\documentclass[a4paper,11pt]{article}
\usepackage[english]{babel}
\usepackage{amsmath,amssymb,graphicx,enumerate,bbm,stmaryrd,mathrsfs,latexsym,theorem}
\usepackage[dvipsnames]{xcolor}
\usepackage[left= 2.5cm, bottom = 4cm, top = 3.5cm, right= 2.5cm]{geometry}
\usepackage[citecolor=blue,colorlinks=true]{hyperref}
\usepackage{authblk} 
\usepackage{tikz}

\renewcommand{\leq}{\leqslant}
\renewcommand{\geq}{\geqslant}
\newcommand{\assign}{:=}
\newcommand{\backassign}{=:}
\newcommand{\precprec}{\prec\!\!\!\prec}
\newcommand{\mathd}{\mathrm{d}}

\newcommand{\tmop}[1]{\ensuremath{\operatorname{#1}}}

\newenvironment{proof}{\noindent\textbf{Proof\ }}{\hspace*{\fill}$\Box$\medskip}
\newtheorem{theorem}{Theorem}[section]
\newtheorem{lemma}[theorem]{Lemma}
\newtheorem{proposition}[theorem]{Proposition}
\newtheorem{corollary}[theorem]{Corollary}
{\theorembodyfont{\rmfamily}\newtheorem{remark}[theorem]{Remark}}

\numberwithin{equation}{section}

\newcommand{\CC}{\mathscr{C} \hspace{.1em}}

\newcommand{\LL}{\mathscr{L}}
\newcommand{\UU}{\mathscr{U}}
\newcommand{\VV}{\mathscr{V}}
\newcommand{\Q}{\mathscr{Q} \hspace{.2em}}

\newcommand{\rmbb}[1]{\textcolor{black}{#1}}
\newcommand{\rmb}[1]{\textcolor{blue}{#1}}
\newcommand{\rmm}[1]{\textcolor{magenta}{#1}}
\newcommand{\rmg}[1]{\textcolor{Bittersweet}{#1}}

\newcommand{\R}{\mathbb{R}}

\newcommand{\N}{\mathbb{N}}
\newcommand{\Td}{{\mathbb{T}^{d}}}
\newcommand{\Prec}{\prec\!\!\!\prec}
\newcommand{\8}{\infty}
\newcommand{\dd}{\mathrm{d}}
\DeclareMathOperator{\esssup}{esssup}

\title{Global solutions to elliptic\\ and parabolic $\Phi^4$ models in
Euclidean space.}

\author[1]{Massimiliano Gubinelli}
\author[2]{Martina Hofmanov\'a}
\affil[1]{\small Hausdorff Center for Mathematics\\ \& Institute for Applied Mathematics\\ University of Bonn\\
Endenicher Allee 60\\
53115 Bonn, Germany.  \href{mailto:gubinelli@iam.uni-bonn.de}{gubinelli@iam.uni-bonn.de} }
\affil[2]{\small Fakult\"at f\"ur Mathematik, Universit\"at Bielefeld, Postfach 10 01 31, 33501 Bielefeld, Germany. \href{mailto:hofmanova@math.uni-bielefeld.de}{hofmanova@math.uni-bielefeld.de}}

\begin{document}
\maketitle

\begin{abstract}
We prove existence of global  solutions to singular SPDEs on $\mathbb{R}^d$ with cubic nonlinearities and additive white noise perturbation, both in the elliptic setting in dimensions $d=4,5$ and in the parabolic setting for $d=2,3$.  We prove uniqueness and coming down from infinity for the parabolic equations. A  motivation for considering  these equations is the construction of scalar interacting Euclidean quantum field theories. The parabolic equations are related to the $\Phi^4_d$ Euclidean quantum field theory via  Parisi--Wu stochastic quantization, while the elliptic equations are linked to the $\Phi^4_{d-2}$ Euclidean quantum field theory via the Parisi--Sourlas dimensional reduction mechanism.  \\[.4em]
  
\end{abstract}

\tableofcontents

\section{Introduction}

This paper is concerned with elliptic and parabolic partial differential equations related to the $\Phi^{4}$ Euclidean quantum field theory on the full space.
More precisely, we consider the following semilinear elliptic partial differential equation on $\R^{d}$ for $d=4,5$,
\begin{equation}\label{eq:phi4}
(-\Delta+\mu)\varphi+\varphi^{3}=\xi,
\end{equation}
where $\xi$ is a space white noise on $\R^{d}$ and $\mu>0$.
We also consider the Cauchy problem for the  semilinear parabolic partial differential equation on $\R_+\times\R^{d}$ with $d=2,3$, given by
\begin{equation}\label{eq:phi4p}
(\partial_{t}-\Delta+\mu)\varphi+\varphi^{3}=\xi,
\end{equation}
where $\xi$ is a space-time white noise on $\R_+\times\R^{d}$ and $\mu\in \R$.

Both equations fall in the category of the so-called \emph{singular} SPDEs, a loose term which means that they are classically ill-posed due to the very irregular nature of the noise $\xi$. Indeed, solutions are expected to take values only in spaces of distributions of negative regularity and the non-linear terms appearing in the equations cannot be given a canonical meaning.   Recent progresses by Hairer~\cite{hairer_theory_2014} and others~\cite{GIP,kupiainen_renormalization_2016,otto_quasilinear_2016} have provided various existence theories for local solutions of the above parabolic equations in a periodic spatial domain. The key idea is to identify suitable subspaces of distributions large enough to contain the candidate solutions and structured enough to allow for the definition of the non-linear terms.  These theories define solutions for the above equations once the non-linear term is \emph{renormalized}, which formally can be understood as a subtraction of an (infinite) correction term:
$$
\varphi^3 \mapsto \varphi^3 - \infty \varphi.
$$
More rigorously, and as we hinted above, this formal expression has to be understood in the sense that even though both terms separately are not well defined, certain combination has a well-defined meaning for a restricted class of distributions $\varphi$. The byproduct of the renormalization is that additional data (in the form of polynomials of the driving noise) have to be considered in order to identify canonically the result of the renormalization. 
It is not the main aim of this paper to discuss the features  of the local solution theory for singular SPDEs as this has been done extensively in the references cited above. 

Our aim  here is to develop a simple global solution theory for equations~\eqref{eq:phi4} and~\eqref{eq:phi4p}. Global solutions rely on specific properties of the equations, in particular here on the \emph{right} sign of the cubic non-linearity. The existence of  global in time solutions of the parabolic equation~\eqref{eq:phi4p} is relevant to the problem of \emph{stochastic quantization} of the $\Phi^4_d$ Euclidean  field theory, that is the measure $\nu$ on distributions over the $d$-dimensional periodic domain $\Lambda=\mathbb{T}^d$ formally given by the Euclidean path integral
\begin{equation}
\label{eq:phi4measure}
\nu(\mathd \phi)= \exp\left[-\int_{\Lambda} \left(\frac12|\nabla \phi|^2+\frac{\mu}{2}\phi^2+\frac14\phi^4\right) \right] \mathd \phi,
\end{equation}
where $\mathbb{T}=\R/2\pi \mathbb{Z}$.
Global in space solutions, that is solutions defined over all $\R^d$ correspond to the infinite volume limit of such a measure. Existence and uniqueness of global space-time solutions for the parabolic model in $d=2$ has been proved by Mourrat and Weber~\cite{MW17}. More recently the same authors have proven existence and uniqueness of global solutions in time on $\mathbb{T}^3$ in~\cite{MWcomedown}. In this last paper they also prove the stronger property, namely, that the solutions \emph{come down from infinity}, meaning that after a finite time the solution belongs to a compact set of the state space uniformly in the initial condition, a very strong property which is entirely due to the presence of the  cubic drift. These results show that singular SPDEs can be used to implement rigorously  the stochastic quantization approach first suggested by Parisi and Wu~\cite{parisi_perturbation_1981} and \emph{construct} random fields sampled according to the measure~\eqref{eq:phi4measure}. Another recent interesting approach which uses the SPDE to construct the measure $\nu$ is that of Albeverio and Kusuoka~\cite{albeverio_invariant_2017} which uses the invariance of approximations and uniform energy estimates on the SPDE to deduce tightness and existence of the limiting $\Phi^4_d$ measure~\eqref{eq:phi4measure}. 

In the present work we complete the picture by proving the global space-time existence and uniqueness for eq.~\eqref{eq:phi4p} in $\R^3$ with an associated \emph{coming down from infinity} property. This will be essentially a byproduct of the technique we develop to analyze the elliptic model~\eqref{eq:phi4} on $\R^d$ with $d=4,5$. The choice of dimensions has a two-fold origin: first it corresponds to the dimensions where the singularities of the elliptic equation match those of the parabolic one for $d=2,3$. Second (and partially related reason) is that there exists a very interesting conjecture of \emph{dimensional reduction} formulated first by Parisi and Sourlas~\cite{parisi_random_1979} which links the behavior of certain SPDEs in $d$ dimensions to that of Euclidean field theories in $d-2$ dimensions. In particular, it is conjectured that the trace on a codimension 2 hyperplane  of solutions to eq.~\eqref{eq:phi4} in $\R^d$ should have the law of the (parabolic) $\Phi^4_{d-2}$ model in $\R^{d-2}$, at least for $d=3,4,5$. This conjecture has been partially validated by rigorous arguments of Klein et al.~\cite{klein_supersymmetry_1983, klein_supersymmetry_1984} in the context of a regularized version of the models. Our study of the singular equation is another step to the full rigorous verification of the dimensional reduction phenomenon. The existence theory of the $d=3$ elliptic model is relatively straightforward and we will not consider it here.

Given the importance of these models in the mathematical physics literature and the open interesting conjectures they are related to, we found essential to devise streamlined arguments to treat global solutions of these equations. The main technical problem with globalization in the solution theory of singular SPDEs is given by the fact that the noise grows at infinity requiring the use of weighted spaces. This in turn requires to exploit fine properties of the equations in order to close the estimates. Witness of the important technical difficulties involved in the global analysis is the \emph{tour de force} that Mourrat and Weber~\cite{MWcomedown} had to put in place to solve the parabolic model on $\mathbb{T}^3$. One of the aim of the present paper is to provide also a simpler proof of their result, proof which is more in line with standard arguments of functional analysis/PDE theory. 
In order to do so we developed a new \emph{localization} technique which allows to split distributions belonging to weighted spaces into an irregular component which behaves nicely at the spatial infinity and a smooth component which grows in space. The localization technique allows to split singular SPDEs in two equations:
\begin{itemize}
\item[-]
 one containing the irregular terms but linear (or almost linear) and not requiring any particular care in the handling of the weighted spaces;
 \item[-] the other containing all the more regular terms and all the non-linearities which can be analyzed using standard PDE arguments, in particular pointwise maximum principle and  pointwise coercive estimates whose weighted version are easy to establish. This avoids the use of weighted $L^p$ spaces and related energy estimates which complicate the analysis of Mourrat and Weber~\cite{MWcomedown} and also of  Albeverio and Kusuoka~\cite{albeverio_invariant_2017}.
\end{itemize}
Other two improvements which we realize in this paper are the following:
\begin{itemize}
\item[a)] we use a direct $L^2$ energy estimate to establish uniqueness for the parabolic model, simplifying the proof and taking full advantage of our $L^\infty$ a priori estimates;
\item[b)] we use a time dependent weight to prove the coming down from infinity, going around the painful induction present in Mourrat and Weber paper and following quite closely the strategy one would adopt for classical driven reaction diffusion equations.
\end{itemize}
A problem which still remains open is that of the global uniqueness in the elliptic setting. Probably uniqueness does not hold or holds only for large masses. This is suggested by the behavior of the corresponding $\Phi^4_{d-2}$ model which is expected to undergo a phase transition at small temperature, corresponding here to a small mass.

\paragraph{Organization of the paper.} In Section \ref{s:prel} we introduce the basic notation and recall  various preliminary results concerning weighted Besov spaces. Then we present interpolation results  and construct the above mentioned localization operators, which are essential in the main body of the paper. As the next step, we establish Schauder and coercive estimates in weighted Besov spaces in both elliptic and parabolic setting and finally we discuss the basic results of the paracontrolled calculus.

In Section \ref{s:proba}, we recall the results of probabilistic analysis connected to the construction of the  stochastic objects needed in the sequel.

Sections \ref{sec:44}, \ref{sec:45} are devoted to the existence for the elliptic $\Phi^{4}$ model in dimension 4 and 5, respectively. More precisely, in the first step, we decompose the equations into  systems of two equations, one irregular  and the other one regular and containing the cubic nonlinearity. The next step is the cornerstone of our  analysis: we derive new a priori estimates for the unknowns of the decomposed system, which are then employed in order to establish existence of solutions. Here we first solve the equations on a large torus using a combination of a variational approach together with the Schaefer's fixed point theorem. Then we let the size of the torus converge to infinity and use compactness.

The a priori estimates from Sections \ref{sec:44}, \ref{sec:45} play the key role in the parabolic setting as well. Namely, in Sections \ref{sec:42}, \ref{sec:43} we study the parabolic $\Phi^{4}$ model in dimension 2 and 3, respectively. We follow a similar decomposition into a system of equations (only with a slight modification in dimension 3) and derive parabolic a priori estimates in analogy to the elliptic situation. These bounds are then used in the proof of existence. However, we proceed differently than in the elliptic setting: we work directly on the full space and mollify the noise, which leads to existence of smooth approximate solutions. The uniform estimates together with a compactness argument allow us to pass to the limit.

In Section \ref{s:uniq} we establish uniqueness of solutions in the parabolic setting. Unlike in the previous sections, it is not enough to work in the $L^{\infty}$-scale of weighted Besov spaces with polynomial weights. In particular, to compensate for the loss of weight in our estimates we employ exponential weights, requiring a different definition of the associated Besov spaces. This is discussed in Section \ref{ss:b}. The  proof of uniqueness then uses solely energy-type estimates in the $L^{2}$-scale of Besov spaces which  takes the full advantage of the well-chosen space-time weight.

Section \ref{s:d} is then concerned with the coming down from infinity property. Here we work with an additional weight in time which vanishes at zero and therefore allows to obtain bounds independent of the initial condition. Such a weight requires careful Schauder and coercive estimates that are established in Sections \ref{ss:d1}, \ref{ss:d2}. The proof of the coming down from infinity  then relies on our approach to a priori estimates from Section \ref{sec:44}, \ref{sec:45} together with a delicate control of the behavior at zero.

Finally, in Appendix \ref{s:a} we collect certain auxiliary results concerning existence for elliptic and parabolic variants of our problem in the smooth setting. Appendix \ref{s:sch} is then devoted to a refined Schauder estimate needed in Section \ref{s:d}.

We point out that for didactic reasons and in order not to blur our arguments, we chose to include in  Section \ref{s:prel}  only the results needed for the existence in Sections ~\ref{sec:44}, \ref{sec:45}, \ref{sec:42}, \ref{sec:43}. Further generalizations are  needed for uniqueness in Section \ref{s:uniq} and for the coming down from infinity in Section \ref{s:d}. The corresponding preliminaries are then discussed directly in the respective sections.

\paragraph{Acknowledgement.} MG is partially supported by the German Research Foundation (DFG) via CRC 1060.

\section{Preliminaries}
\label{s:prel}

\subsection{Weighted Besov spaces}
\label{ssec:besov}

As the first step, we introduce weighted Besov spaces which will be used in the sequel. Recall that the collection of  admissible weight functions is the collection of all positive $C^\infty(\mathbb{R}^d)$ functions $\rho$ with the following properties:
\begin{enumerate}
\item For all $\gamma\in\mathbb{N}_0^d$ there is a positive constant $c_\gamma$ with
$$
|D^\gamma\rho(x)|\leqslant c_\gamma \rho(x),\qquad\text{for all }x\in\mathbb{R}^d.
$$
\item There are two constants $c>0$ and $b\geq 0$ such that
$$
0<\rho(x)\leqslant c\rho(y)(1+|x-y|^2)^{b/2},\qquad\text{for all }x,y\in\mathbb{R}^d.
$$
\end{enumerate}

The space of Schwartz functions on $\R^{d}$ is denoted by $\mathcal{S}(\R^{d})$ and its dual, the space of tempered distributions is $\mathcal{S}'(\R^{d})$. The Fourier transform of $u\in \mathcal{S}'(\R^{d})$ is given by
$$
\mathcal{F}u(z)=\int_{\R^{d}}u(x)e^{-i z\cdot x}\,\dd x,
$$
so that the inverse Fourier transform is given by $\mathcal{F}^{-1}u(x)=(2\pi)^{-d}\mathcal{F}u(-x).$
 By $(\Delta_{i})_{i\geq -1}$ we denote the Littlewood--Paley blocks corresponding to a dyadic partition of unity.
If $\rho$ is an admissible weight and $\alpha\in\mathbb{R}$, we define the weighted Besov space $B^\alpha_{\infty,\infty}(\rho)\backassign \CC^\alpha(\rho)$ as the collection of all $f\in\mathcal{S}'(\R^{d})$ with finite norm
\[ \| f \|_{\CC^{\alpha} (\rho)} = \sup_{i\geq -1} 2^{i \alpha} \| \Delta_i f
   \|_{L^{\infty} (\rho)} = \sup_{i\geq -1} 2^{i \alpha} \| \rho \Delta_i f
   \|_{L^{\infty}} . \]
More details can be found  e.g. in \cite{T06}. 
Particularly,  due to \cite[Theorem 6.5]{T06}, it holds true that
\begin{equation}\label{eq:1}
   \|f\|_{\CC^\alpha(\rho)}\sim\|\rho f\|_{\CC^\alpha}
\end{equation}
   in the sense of equivalence of norms, where the latter denotes the norm in the classical (unweighted) Besov space $\CC^\alpha=B^\alpha_{\infty,\infty}(\R^{d})$. Moreover, it was shown in \cite[Theorem 6.9]{T06} that for $\alpha\in(0,M)$ with $M\in\N$, the weighted space $\CC^\alpha(\rho)$ admits an equivalent norm given by
\begin{equation}\label{eq:2h}
\|f\|_{L^\infty(\rho)}+\sup_{0<|h|\leqslant 1}|h|^{-\alpha}\|\Delta^M_h f\|_{L^\infty(\rho)},
\end{equation}
where $\Delta^M_h$ is the $M^{th}$-order finite difference operator  defined inductively by
$$
(\Delta^1_hf)(x)=f(x+h)-f(x),\qquad (\Delta^{\ell+1}_h)f(x)=\Delta^1_h(\Delta^{\ell}_hf)(x),\qquad\ell\in\N,
$$
Introduce a partition of unity  $\sum_{m\in\mathbb{Z}^{d}} \Lambda_m = 1$, where $\Lambda_{m}(x):=\Lambda(x-m)$ for a compactly supported $C^{\8}$-function $\Lambda$ on $\R^{d}$ and $m\in\mathbb{Z}^{d}$. Then the following localization principle for weighted Besov spaces follows from \eqref{eq:1} and \cite[Theorem 2.4.7]{T92}: let $\alpha\in\R$ then
\begin{equation}\label{eq:lem2.1}
   \|f\|_{\CC^\alpha(\rho)}\sim\sup_{m\in\mathbb{Z}^{d}}\|\Lambda_{m} f\|_{\CC^\alpha(\rho)}
\end{equation}
holds true in the sense of equivalence of norms. For most of our purposes, the following result in the case $\alpha>0$ will be sufficient. Let  $\sum_{k\geq-1} w_k = 1$ be a smooth partition of unity in spherical dyadic slices
where $w_{-1}$ is supported in a ball containing zero and there exists an annulus $\mathcal{A}=\{x\in\R^{d};a\leq|x|\leq b\}$ for some $0<a<b$ such that each $w_k$ for $k\geq0$ is supported in the annulus $2^k\mathcal{A}$. Set
$\tilde{w}_k = \sum_{\substack{i\geq-1\\i \sim k}} w_i$, where we write $i\sim k$ provided $\mathrm{supp}\,  w_{i}\cap \mathrm{supp}\, w_{k}\neq \emptyset.$

\begin{lemma}\label{lem:2.1}
It holds true that
\[ \| f \|_{L^{\infty} (\rho)} \leqslant \sup_{k\geq-1} \| \tilde{w}_k f
   \|_{L^{\infty} (\rho)}, \]
and if $\alpha>0$ then also
   $$
   \| f \|_{\CC^{\alpha} (\rho)} \lesssim\sup_{k\geq-1} \| \tilde{w}_k f
   \|_{\CC^{\alpha} (\rho)}.
   $$
\end{lemma}

\begin{proof}
Due to the construction of $(\tilde w_k)_{k\geq -1}$, for every $x\in\R^d$ there exists $k\geq -1$ such that $f(y)=\tilde{w}_k(y)f(y)$ for all $y\in \R^d$ with $|x-y|<1$. Consequently, the first claim follows. To show the second one, let $M\in\N$ be the smallest integer such that $\alpha<M$.
Then, it can be observed that, in addition to \eqref{eq:2h}, also 
$$
\| f \|_{L^{\infty} (\rho)} +\sup_{0<|h|< \frac{1}{M}}h^{-\alpha}\|\Delta^M_h f\|_{L^\infty(\rho)}
$$
defines an equivalent norm on $\CC^\alpha(\rho)$. The first summand is estimated as in the previous step. For the second summand, consider $|h|<\frac{1}{M}$. Since $(\Delta^M_hf)(x)$ depends only on  values  $f(y)$ for  $|y-x|\leqslant M|h|<1$, we deduce that for every $x\in \R^d$ there exists $k\in\N_0$ such that $f(y)=\tilde w_k(y)f(y)$ whenever $|y-x|<1$ and consequently also $(\Delta^M_hf)(y)=(\Delta^M_h(\tilde w_kf))(y)$. Thus
$$
\sup_{0<|h|< \frac{1}{M}}h^{-\alpha}\|\Delta^M_h f\|_{L^\infty(\rho)}\leqslant \sup_{k\geq-1}\sup_{0<|h|< \frac{1}{M}}h^{-\alpha}\|\Delta^M_h (\tilde w_kf)\|_{L^\infty(\rho)}
$$
and the second claim follows.
\end{proof}

Throughout this paper $\rho$ stands for a weight which is admissible and either constant or decreasing at infinity. It depends only on the space variable in the case of elliptic problems  or on both space and time for parabolic equations. We will not repeat the word ``admissible'' in the sequel. Moreover,  we will often work with  polynomial weights of the form $\rho(x)=\langle x\rangle^{-\nu}$ where  $\langle x\rangle=(1+|x|^{2})^{1/2}$ and $\nu\geq 0$. In the same spirit we will consider space-time dependent polynomial weights or $\rho(t,x)=\langle (t,x)\rangle^{-\nu}=(1+|(t,x)|^{2})^{-\nu/2}$ for $\nu\geq 0$. In addition, certain non-admissible weights will be needed in Section \ref{s:uniq} and Section \ref{s:d}. Namely, the proof of uniqueness in Section \ref{s:uniq} employs a weight that vanishes exponentially at infinity and consequently the definition of the associated Besov spaces cannot be based on Schwartz functions but rather on the so-called Gevrey classes as discussed in \cite{MW17}. The coming down from infinity property in Section~\ref{s:d} then requires a weight in time that vanishes in zero and is therefore also not an admissible weight in the sense of the above definition. The necessary results for these particular weights are discussed in Section~\ref{ss:b} and Sections \ref{ss:d0},  \ref{ss:d1}, \ref{ss:d2}.

Let $\rho$ be a polynomial space-dependent weight. Then the following embedding holds true
\begin{equation}\label{eq:emb}
\CC^{\beta_1}(\rho^{\gamma_1})\subset \CC^{\beta_2}(\rho^{\gamma_2}) \quad\text{provided}\quad \beta_1\geq\beta_2,\ \gamma_1\leqslant \gamma_2,
\end{equation}
and, according to \cite[Theorem 6.31]{T06}, the embedding in \eqref{eq:emb} is compact provided  $ \beta_1>\beta_2$ and $\gamma_1< \gamma_2$.

For parabolic equations, we will also need weighted   function spaces of space-time dependent functions/distributions. Let $\rho$ be a polynomial space-time weight and $\alpha\in\R$ and denote $\rho_{t}(\cdot)=\rho(t,\cdot)$, $t\in[0,\8)$.  Then $C\CC^{\alpha}(\rho)$ is the space of space-time distributions $f$ that are continuous in time, satisfy  $f(t)\in\CC^{\alpha}(\rho_{t})$ for every $t\in[0,\infty)$, and have finite norm
\[
\|f\|_{C\CC^{\alpha}(\rho)}:=\sup_{t\geq 0} \|(\rho f)(t)\|_{\CC^{\alpha}}.
\]
If a mapping $f:[0,\infty)\to  \CC^{\alpha}(\rho_{0}) $ is only bounded but not continuous, we write $f\in L^{\infty}\CC^{\alpha}(\rho)$ with the norm
\[
\|f\|_{L^{\infty}\CC^{\alpha}(\rho)}:=\esssup_{t\geq 0} \|(\rho f)(t)\|_{\CC^{\alpha}}<\8.
\]
Time regularity will be measured in terms of classical H\"older norms. In particular, for $\alpha\in(0,1)$ and $\beta\in\R$ we denote by $C^{\alpha}\CC^{\beta}(\rho)$  the space of mappings $f:[0,\infty)\to  \CC^{\beta}(\rho_{0}) $ with finite norm
$$
\|f\|_{C^{\alpha}\CC^{\beta}(\rho)}:=\sup_{t\geq 0} \|(\rho f)(t)\|_{\CC^{\beta}}+\sup_{s,t\geq 0,s\neq t} \frac{\|(\rho f)(t)-(\rho f)(s)\|_{\CC^{\beta}}}{|t-s|^{\alpha}}.
$$
It can be seen (cf. \eqref{eq:2h}) that since $\rho$ is a polynomial weight, this norm is equivalent to
\begin{equation}\label{eq:teq}
\|f\|_{C^{\alpha}\CC^{\beta}(\rho)}\sim\sup_{t\geq 0} \|(\rho f)(t)\|_{\CC^{\beta}}+\sup_{s,t\geq 0,s\neq t} \frac{\|\rho_{t}(f(t)- f(s))\|_{\CC^{\beta}}}{|t-s|^{\alpha}}.
\end{equation}
Similarly, we define the space $C^{\alpha}L^{\8}(\rho)$.

In the case we consider only a finite time interval $[0,T]$, for some $T>0$, and a time-independent weight $\rho$, we write $f\in C_{T}\CC^{\alpha}(\rho)$, $f\in L^{\infty}_{T}\CC^{\alpha}(\rho)$, $C_{T}^{\alpha}\CC^{\beta}(\rho)$ and $C_{T}^{\alpha}L^{\8}(\rho)$ with straightforward modifications in the corresponding norms.

\subsection{Interpolation}

We present a simple interpolation result for  weighted Besov spaces.

\begin{lemma}
  \label{lemma:interp}
 Let $\kappa\in (0,1)$ and let $\rho$ be a space-time weight.  We have, for any $\alpha \in [0, 2 + \kappa]$
  \[ \| \psi \|_{\CC^{\alpha} (\rho^{1 + \alpha})} \lesssim \| \psi
     \|_{L^{\infty} (\rho)}^{1 - \alpha / (2 + \kappa)} \| \psi \|_{\CC^{2 +
     \kappa} (\rho^{3 + \kappa})}^{\alpha / (2 + \kappa)} .\]
\end{lemma}

\begin{proof}
 It holds
  \[ \|  \Delta_k \psi \|_{L^{\infty} (\rho^{1 +
     \alpha})} \lesssim \|  \rho^{1 + \alpha} \Delta_k \psi
     \|_{L^{\infty}} \lesssim \|  \rho \Delta_k \psi
     \|_{L^{\infty}}^{1 - \alpha / (2 + \kappa)} \|  \rho^{3 +
     \kappa} \Delta_k \psi \|_{L^{\infty}}^{\alpha / (2 + \kappa)} \]
  \[ \lesssim \| \psi \|_{L^{\infty} (\rho)}^{1 - \alpha /
     (2 + \kappa)} \| \Delta_k \psi \|_{L^{\infty}
     (\rho^{3 + \kappa})}^{\alpha / (2 + \kappa)} \lesssim 2^{- \alpha k} \|
\psi \|_{ L^{\infty} (\rho)}^{1 - \alpha / (2 + \kappa)} \|
 \psi \|_{\CC^{2 + \kappa} (\rho^{3 +
     \kappa})}^{\alpha / (2 + \kappa)} \]
  which proves the  claim.
\end{proof}

We will also need the following version  adapted to time-dependent problems.

\begin{lemma}
  \label{lemma:interp2}
Let $\kappa\in (0,1)$ and let $\rho$ be a space-time weight.  We have, for any $\alpha \in [0, 2 + \kappa]$
  \[ \| \psi \|_{C\CC^{\alpha} (\rho^{1 + \alpha})} \lesssim \| \psi
     \|_{L^{\8}L^{\infty} (\rho)}^{1 - \alpha / (2 + \kappa)} \| \psi \|_{C\CC^{2 +
     \kappa} (\rho^{3 + \kappa})}^{\alpha / (2 + \kappa)} .\]
Moreover, if  $\alpha/2\notin\N_{0}$ then
       \[ \| \psi \|_{C^{\alpha/2}L^{\8} (\rho^{1 + \alpha})} \lesssim \| \psi
     \|_{L^{\8}L^{\infty} (\rho)}^{1 - \alpha / (2 + \kappa)} \| \psi \|_{C^{(2+\kappa)/2}_{b}L^{\8} (\rho^{3 + \kappa})}^{\alpha / (2 + \kappa)} .\]
\end{lemma}

\begin{proof}
The first claim is a straightforward modification of Lemma \ref{lemma:interp}.
The second one can be obtained by the same approach since for $\alpha/2\notin\N_{0}$ the H\"older space $C^{\alpha/2}$ can be identified with the Besov space $B^{\alpha/2}_{\8,\8}$ and functions in $C^{\alpha/2}L^{\8}(\rho)$ can be naturally extended to be defined on the full space $\R\times\R^{d}$ while preserving the same norm.
\end{proof}

\subsection{Localization operators}
\label{ssec:local}

Here we construct localization operators $\UU_{>},
\UU_{\leqslant}$ which play the key role in our analysis. These localizers allow to decompose  a distribution $f$ into a sum of two components: one belongs to a (weighted) Besov space of higher regularity whereas the other one is less regular. To this end, let  $\sum_{k\geq-1} w_k = 1$ be a smooth dyadic  partition of unity on $\R^{d}$
where $w_{-1}$ is supported in a ball containing zero and there exists an annulus $\mathcal{A}=\{x\in\R^{d};a\leq|x|\leq b\}$ for some $0<a<b$ such that each $w_k$ for $k\geq0$ is supported in the annulus $2^k\mathcal{A}$.
Let $(L_{k})_{k\geq -1}\subset [-1,\infty)$ be a sequence of real numbers and let $f\in\mathcal{S}'(\R^{d})$. We define the localization operators by
\[ \UU_{>} f = \sum_k w_k \Delta_{> L_k} f, \qquad
   \UU_{\leqslant} f = \sum_k w_k \Delta_{\leqslant L_k} f, \]
where $\Delta_{>L_{k}}=\sum_{j:j>L_{k}}\Delta_{j}$ and $\Delta_{\leq L_{k}}=\sum_{j:j\leq L_{k}}\Delta_{j}$. We point out that in the sequel, we will use various localizing sequence $(L_{k})_{k\geq -1}$, depending on the context. However, for notational simplicity, we
will not denote these operators by different symbols.

\begin{lemma}\label{lem:local}
\rmbb{Let $L>0$ be given. There exists a  choice of parameters $(L_k)_{k\geq -1}$ such that for all $\alpha,\delta,\gamma>0$ and $a,b\in\R$ it holds true
      \[ \| \UU_{>} f \|_{\CC^{-\alpha - \delta} (\rho^{- a})} \lesssim 2^{- \delta L} \| f \|_{\CC^{-\alpha}
     (\rho^{-a+\delta})}, \qquad \| \UU_{\leqslant} f \|_{\CC^{-\alpha+\gamma}
     (\rho^{b})} \lesssim 2^{\gamma	
     L} \| f \|_{\CC^{-\alpha} (\rho^{b-\gamma})}, \]
          where the proportionality constant depends on $\alpha,\delta,\gamma,a,b$ but is independent of $f$.}
\end{lemma}

\begin{proof}
  Denote $c_k = - \log_2 \|  w_k \|_{L^{\infty} (\rho)}$ and let $\beta>\alpha+\delta$. Then we have
  \[ \|  w_k \|_{\CC^{\beta} (\rho^{\gamma})} \simeq \|  w_k \|_{L^{\infty}
     (\rho)}^{\gamma} = 2^{- \gamma c_k}, \]
  \[ \|  w_k \|_{\CC^{\beta} (\rho^{- \delta})} \simeq \|  w_k
     \|_{L^{\infty} (\rho)}^{- \delta} = 2^{\delta c_k} . \]
According to \eqref{eq:lem2.1}    and since there exists $M\in\N$ such that for every $m\in\mathbb{Z}^{d}$  the support of $\Lambda_{m}$ intersects the support of $w_{k}$ only for $k\in A_{m}$, where $A_{m}$ in a set of cardinality at most $M$, it holds 
  \[ \| \UU_{>} f \|_{\CC^{-\alpha - \delta} (\rho^{- a})} \lesssim
     \sup_{m\in\mathbb{Z}^{d}}\|\Lambda_{m}\UU_{>}f\|_{\CC^{ -\alpha - \delta} (\rho^{- a})}\lesssim \sup_{m\in\mathbb{Z}^{d}}\|\Lambda_{m}\sum_{k\in A_{m}}w_{k}\Delta_{>L_{k}}f\|_{\CC^{-\alpha - \delta} (\rho^{- a})}\]
 \[  \lesssim_{M} \sup_{k}\|  w_k \Delta_{> L_k} f \|_{\CC^{ -\alpha - \delta} (\rho^{- a})}
     \lesssim \sup_k \|  w_k \|_{\CC^{\beta} (\rho^{-\delta})}
     \| \Delta_{> L_k} f \|_{\CC^{ -\alpha - \delta} (\rho^{-a+\delta})} \]
  \[ \lesssim \sup_k 2^{\delta c_k - \delta L_k} \| f \|_{\CC^{- \alpha}
     (\rho^{-a+\delta})} \lesssim 2^{- \delta L} \| f \|_{\CC^{- \alpha}
     (\rho^{-a+\delta})}, \]
where we set $L_k =  c_k+ L$.
  On the other hand, the same argument implies
  \[ \| \UU_{\leqslant} f \|_{\CC^{-\alpha+\gamma} (\rho^b)} \lesssim \sup_k
     \|  w_k \Delta_{\leqslant L_k} f \|_{\CC^{-\alpha+\gamma} (\rho^b)} \lesssim
     \sup_k \|  w_k \|_{\CC^{\beta} (\rho^{\gamma})} \| \Delta_{\leqslant
     L_k} f \|_{\CC^{-\alpha+\gamma} (\rho^{b-\gamma})} \]
  \[ \lesssim \sup_k 2^{\gamma L_k - \gamma c_k} \| f
     \|_{\CC^{-\alpha} (\rho^{b-\gamma})} \lesssim 2^{\gamma L} \| f \|_{\CC^{- \alpha}
     (\rho^{b-\gamma})}.\]
\end{proof}

\begin{remark}
     Note that the sequence $(L_k)_{k\geq -1}$ in Lemma \ref{lem:local} does not depend on any of the parameters $\alpha,\delta,\kappa,a,b$ nor on the function $f$.

\end{remark}

We will also need the following version adapted to time-dependent problems. Let $(v_{\ell})_{\ell\geq -1}$ be a smooth dyadic partition of unity on $[0,\8)$ such that  $v_{-1}$ is supported in a ball containing zero and there exists an annulus $\mathcal{A}=\{t\in[0,\infty);a\leq t\leq b\}$ for some $0<a<b$ such that each $v_\ell$ for $\ell\geq0$ is supported in the annulus $2^k\mathcal{A}$. 
Let $\tilde v_{\ell}=\sum_{i:i\sim\ell}v_{i}$.
 For a given sequence $(L_{k,\ell})_{k,\ell\geq -1}$  we define localization operators $\VV_{>},
\VV_{\leqslant}$  by
\begin{equation}\label{eq:locd}
\VV_{>} f = \sum_{k,\ell} v_{\ell}w_k \Delta_{> L_{k,\ell}} f, \qquad
   \VV_{\leqslant} f = \sum_{k,\ell} v_{\ell}w_k \Delta_{\leqslant L_{k,\ell}} f.
   \end{equation}

\begin{lemma}\label{lem:local2}
Let $L>0$ be given and  let $\rho$ be a space-time weight. There exists a  choice of parameters $(L_{k,\ell})_{k,\ell\geq -1}$ such that for all $\alpha,\delta,\gamma>0$ and $a,b\in\R$ it holds true
      \[ \| \VV_{>} f \|_{C\CC^{- \alpha - \delta} (\rho^{- a})} \lesssim 2^{- \delta L} \| f \|_{C\CC^{- \alpha}
     (\rho^{-a+\delta})}, \qquad \| \VV_{\leqslant} f \|_{C\CC^{-\alpha+\gamma}
     (\rho^{b})} \lesssim 2^{\gamma
     L} \| f \|_{C\CC^{- \alpha} (\rho^{b-\gamma})}, \]
          where the proportionality constant depends on $\alpha,\delta,\kappa,a,b$ but is independent of $f$.
\end{lemma}

\begin{proof}
Similarly to the proof of Lemma \ref{lemma:interp2} we denote $c_{k,\ell} = - \log_2 \|\tilde v_{\ell}  w_k \|_{CL^{\infty} (\rho)}$ and let $\beta>\alpha+\delta$.
  Then we have
  \[ \|\tilde v_{\ell}  w_k \|_{C\CC^{\beta} (\rho^{\gamma})} \simeq \|\tilde v_{\ell}  w_k \|_{CL^{\infty}
     (\rho)}^{\gamma} = 2^{- \gamma c_{k,\ell}}, \]
  \[ \|\tilde v_{\ell}  w_k \|_{C\CC^{\beta} (\rho^{- \delta})} \simeq \|\tilde v_{\ell}  w_k
     \|_{CL^{\infty} (\rho)}^{- \delta} = 2^{\delta c_{k,\ell}} . \]
In view of  \eqref{eq:lem2.1} and Lemma \ref{lem:2.1}, we deduce (similarly to the proof of Lemma \ref{lem:local}) that
  \[ \| \VV_{>} f \|_{C\CC^{- \alpha - \delta} (\rho^{- a})} \lesssim
     \sup_{k,\ell} \|\tilde v_{\ell}  w_k \Delta_{> L_{k,\ell}} f \|_{C\CC^{- \alpha - \delta} (\rho^{- a})}
    \]
    \[ \lesssim \sup_{k,\ell} \|\tilde v_{\ell}  w_k \|_{C\CC^{\beta} (\rho^{-\delta})}
     \| \Delta_{> L_{k,\ell}} f \|_{C\CC^{- \alpha - \delta} (\rho^{-a+\delta})} \]
  \[ \lesssim \sup_{k,\ell} 2^{\delta c_{k,\ell} - \delta L_{k,\ell}} \| f \|_{C\CC^{- \alpha}
     (\rho^{-a+\delta})} \lesssim 2^{- \delta L} \| f \|_{C\CC^{- \alpha}
     (\rho^{-a+\delta})}, \]
where we set $L_{k,\ell} =  c_{k,\ell}+ L$.
  On the other hand, it holds
  \[ \| \VV_{\leqslant} f \|_{C\CC^{-\alpha+\gamma} (\rho^b)} \lesssim \sup_{k,\ell}
     \|\tilde v_{\ell}  w_k \Delta_{\leqslant L_{k,\ell}} f \|_{C\CC^{-\alpha+\gamma} (\rho^b)} \lesssim
     \sup_{k,\ell} \|\tilde v_{\ell}  w_k \|_{C\CC^{\beta} (\rho^{\gamma})} \| \Delta_{\leqslant
     L_{k,\ell}} f \|_{C\CC^{-\alpha+\gamma} (\rho^{b-\gamma})} \]
  \[ \lesssim \sup_{k,\ell} 2^{\gamma L_{k,\ell} - \gamma c_{k,\ell}} \| f
     \|_{C\CC^{- \alpha} (\rho^{b-\gamma})} \lesssim 2^{\gamma L} \| f \|_{C\CC^{- \alpha}
     (\rho^{b-\gamma})}.\]
\end{proof}

\subsection{Elliptic Schauder estimates}
\label{s:ell1}

We proceed with Schauder estimates valid for elliptic partial differential equations with cubic nonlinearities. Throughout the paper, we denote $\Q=-\Delta+\mu$.

\begin{lemma}
  \label{lemma:schauder-ellptic}Fix $\kappa > 0$ and let $\psi \in \CC^{2 +
  \kappa} (\rho^{3 + \kappa}) \cap L^{\infty} (\rho)$ be a classical solution
  to
  \[ \Q \psi + \psi^3 = \Psi \]
  then
  \[ \| \psi \|_{\CC^{2 + \kappa} (\rho^{3 + \kappa})} \lesssim_{\rho,\mu}  \|
     \Psi \|_{\CC^{\kappa} (\rho^{3 + \kappa})} + \| \psi \|_{L^{\infty}
     (\rho)}^{3 + \kappa}. \]
\end{lemma}

\begin{proof}
In view of \cite[Theorem 6.5]{T06} it holds
$$
\|\Q f\|_{\CC^\alpha(\rho)}\simeq_{\mu} \|f\|_{\CC^{2+\alpha}(\rho)}
$$
in the sense of equivalence of norms. Hence
$$
\|\psi\|_{\CC^{2+\kappa}(\rho^{3+\kappa})}\lesssim \|\Q\psi\|_{\CC^\kappa(\rho^{3+\kappa})}\leqslant \|\Psi\|_{\CC^\kappa(\rho^{3+\kappa})}+\|\psi^3\|_{\CC^\kappa(\rho^{3+\kappa})}
$$
and we estimate  using Lemma \ref{lemma:interp} and weighted Young inequality 
  \[ \| \psi^3 \|_{\CC^{\kappa} (\rho^{3 + \kappa})} \lesssim \| \psi
     \|_{L^{\infty} (\rho)}^2 \| \psi \|_{\CC^{\kappa} (\rho^{1 + \kappa})}\lesssim \| \psi
     \|_{L^{\infty} (\rho)}^2 \| \psi \|_{L^{\infty} (\rho)}^{1 - \kappa / (2
     + \kappa)} \| \psi \|_{\CC^{2 + \kappa} (\rho^{3 + \kappa})}^{\kappa / (2
     + \kappa)}
    \]
    \[
    \leq
    c\left( \|\psi\|_{L^{\infty} (\rho)}^2 \| \psi \|_{L^{\infty} (\rho)}^{1 - \kappa / (2
     + \kappa)}\right)^{(2 + \kappa) / 2} + \frac{1}{2}\left( \| \psi \|_{\CC^{2 +
     \kappa} (\rho^{3 + \kappa})}^{\kappa / (2 + \kappa)} \right)^{(2 +
     \kappa) / \kappa}
     \]
     \[
     \leqslant c\| \psi
     \|_{L^{\infty} (\rho)}^{3 + \kappa} +
     \frac{1}{2} \| \psi \|_{\CC^{2 + \kappa} (\rho^{3 + \kappa})} .
    \]
Thus we finally deduce the claim.
\end{proof}

\subsection{Elliptic coercive estimates}
\label{s:ell2}

An essential result in our analysis is the following maximum principle in the weighted setting.

\begin{lemma}
  \label{lemma:apriori-elliptic}
  Fix $\kappa > 0$ and let $\psi \in \CC^{2 +
  \kappa} (\rho^{3 + \kappa}) \cap L^{\infty} (\rho)$ be a classical solution
  to
  \[ \Q \psi + \psi^3 = \Psi. \]
  Then the following a priori estimate holds
  \[ \| \psi \|_{L^{\infty} (\rho)} \lesssim_{\rho,\mu} 1 + \| \Psi \|_{L^{\infty}
     (\rho^3)}^{1 / 3}. \]
\end{lemma}

\begin{proof}
Let $\rho > 0$ be the weight from the statement of the Lemma and let $\bar{\psi} = \rho \psi$. Due to the assumption, $\bar{\psi}$ is bounded and locally belongs  to $\CC^{2+\kappa}$. 
  Assume for a moment that $\bar{\psi}$ has a global maximum and let $\hat{x}$
  be a global maximum point of $\bar{\psi}$. Then at $\hat{x}$ we have
  \[ 0 = \nabla \bar{\psi} = \rho \nabla \psi + \psi \nabla \rho, \qquad 0
     \leqslant - \Delta \bar{\psi} = - \rho \Delta \psi - (\Delta \rho) \psi -
     2 \nabla \rho \nabla \psi = - \rho \Delta \psi - \left[ (\Delta \rho) - 2
     \frac{| \nabla \rho |^2}{\rho} \right] \psi, \]
  so always at $\hat{x}$ we also have
  \begin{equation*}
   \psi^3 + \mu \psi \leqslant \Psi - \left[ \frac{(\Delta \rho)}{\rho} - 2
     \frac{| \nabla \rho |^2}{\rho^2} \right] \psi 
  \end{equation*}
and  multiplying by $\rho^3$ leads to
  \[ (\bar{\psi})^3 \leqslant \rho^3 \Psi - \rho^2 (\mu + [(\Delta \rho) /
     \rho - 2 (\nabla \rho / \rho)^2]) \bar{\psi}. \]
If $\bar{\psi}(\hat{ x})\geq 0$     then
  $(\bar{\psi}^{})_+^3 \leqslant \| \rho^3 \Psi
  \|_{L^{\infty}} + C_{\rho,\mu} \| \rho^2 \bar{\psi} \|_{L^{\infty}}$. A similar
  reasoning at minima gives  $(\bar{\psi}^{})_-^3 \leqslant \| \rho^3 \Psi
  \|_{L^{\infty}} + C_{\rho,\mu} \| \rho^2 \bar{\psi} \|_{L^{\infty}}$, hence
  \[ \| \psi \|_{L^{\infty} (\rho)} \leqslant \| \Psi \|_{L^{\infty}
     (\rho^3)}^{1 / 3} + C_{\rho,\mu} \| \psi \|_{L^{\infty} (\rho)}^{1 / 3} .
  \]
Using weighted Young inequality we can absorb the second term of the r.h.s. into the l.h.s. and conclude that
  \[ \| \psi \|_{L^{\infty} (\rho)} \lesssim_{\rho,\mu} 1 + \| \Psi \|_{L^{\infty}
     (\rho^3)}^{1 / 3} . \]

Next, we consider the situation when $\psi\rho$ does not attain its global maximum. Since $\psi\rho$ is smooth and  bounded on $\R^{d}$ due to the assumption, it follows that  $\psi\rho^{1+\delta}$ vanishes at infinity for every $\delta\in (0,1)$. Consequently, it has a global maximum point and the previous part of the proof applies with $\rho$ replaced by $\rho^{1+\delta}$. The conclusion then follows by sending $\delta\to 0$ since the corresponding constant $c_{\rho^{1+\delta},\mu}$ is bounded uniformly in $\delta\in (0,1).$
\end{proof}

\subsection{Parabolic Schauder estimates}

As the next step, we derive a parabolic analog of Section \ref{s:ell1}.
To this end,  we first observe that the following Schauder estimates hold true in the weighted Besov spaces. They can be proved similarly to  \cite[Lemma A.9]{GIP}, see also \cite[Section 3.2]{MW17}.

\begin{remark}\label{rem:mu}
We note that the Schauder estimates below are formulated for a positive mass $\mu>0$. However, it can be observed that for the parabolic $\Phi^{4}$ model studied in Sections \ref{sec:42}, \ref{sec:43}, \ref{s:uniq}, \ref{s:d} this does not bring any loss of generality. Indeed, we may always add a linear term with positive mass to both sides of the equation and consider the original massive term as a right hand side. This is not true for the elliptic $\Phi^{4}$ model where the positivity of the mass seems to be essential. For notational simplicity we therefore adopt the convention that $\mu>0$ throughout the paper, that is,  for both elliptic and parabolic equations.
\end{remark}

Recall that we denoted $\Q=-\Delta+\mu$ and let $\LL=\partial_{t}+\Q$. This notation will be used throughout the paper.

\begin{lemma}\label{lemma:sch}
Let $\mu>0$, $\alpha\in\R$ and let $\rho$ be a  space-time weight. Let $v$ and $w$ solve, respectively,
$$
\LL v=f,\quad v(0)=0,\quad\qquad \LL  w=0,\quad w(0)=w_{0}.
$$ 
Then  it holds uniformly over $t\geq 0$
\begin{equation}\label{eq:1vw}
\|v(t)\|_{\CC^{2+\alpha}(\rho_{t})}\lesssim \|f\|_{L^{\infty}\CC^{\alpha}(\rho)},\qquad
\|w(t)\|_{\CC^{2+\alpha}(\rho_{t})}\lesssim \|w_{0}\|_{\CC^{2+\alpha}(\rho_{0})},
\end{equation}
if $0\leqslant 2+\alpha < 2$ then
$$
\|v\|_{C^{(2+\alpha)/{2}}L^{\infty}(\rho)}\lesssim \|f\|_{L^{\infty}\CC^{\alpha}(\rho)},\qquad\|v\|_{C^{1}L^{\infty}(\rho)}\lesssim \|f\|_{C\CC^{0}(\rho)},
$$

$$
 \|w\|_{C^{(2+\alpha)/2}L^{\infty}(\rho)}\lesssim \|w_{0}\|_{\CC^{2+\alpha}(\rho_{0})}.
$$
\end{lemma}

\begin{proof}
Denote $P_{t}=e^{t(\Delta-\mu)}$ be the semigroup of operators generated by $\Delta-\mu$ and recall that  $\mu>0$. Consider a time independent weight $\rho$ and  observe that similarly to \cite[Lemma A.7, Lemma A.8]{GIP} it holds true  uniformly over $t\geq 0$
\begin{align*}
\|P_{t}g\|_{\CC^{\alpha+\delta}(\rho)}\lesssim e^{-\mu t}t^{-\delta/2}\|g\|_{\CC^{\alpha}(\rho)},\qquad \|P_{t}g\|_{\CC^{\delta}(\rho)}\lesssim e^{-\mu t} t^{-\delta/2}\|g\|_{L^{\8}(\rho)}
\end{align*}
and if $0\leq\alpha \leqslant 2$
$$
\|(P_{t}-\mathrm{Id})g\|_{L^{\8}(\rho)}\leqslant |e^{-\mu t}-1|\|e^{t\Delta}g \|_{L^{\8}(\rho)}+e^{-\mu t}\|(e^{t\Delta}-\mathrm{Id})g \|_{L^{\8}(\rho)} \lesssim t^{\alpha/2}\|g\|_{\CC^{\alpha}(\rho)}.
$$
For a space-time weight, we obtain by the same  argument 
\begin{align}\label{eq:sch10}
\|P_{t}g\|_{\CC^{\alpha+\delta}(\rho_{t})}\lesssim e^{-\mu t}t^{-\delta/2}\|g\|_{\CC^{\alpha}(\rho_{t})},\qquad \|P_{t}g\|_{\CC^{\delta}(\rho_{t})}\lesssim e^{-\mu t} t^{-\delta/2}\|g\|_{L^{\8}(\rho_{t})},
\end{align}
$$
\|(P_{t}-\mathrm{Id})g\|_{L^{\8}(\rho_{t})} \lesssim t^{\alpha/2}\|g\|_{\CC^{\alpha}(\rho_{t})}.
$$
Then, if $2^{-2k}>t$  it follows from the fact that the weight is nonincreasing in time that
\begin{align*}
\left\|\int_{0}^{t}P_{t-s}\Delta_{k}f_{s}\mathrm{d}s\right\|_{L^{\infty}(\rho_{t})}&\lesssim t 2^{-k\alpha}\|f\|_{L^{\infty}\CC^{\alpha}(\rho)}\lesssim 2^{-k(2+\alpha)}\|f\|_{L^{\infty}\CC^{\alpha}(\rho)}.
\end{align*}
If $2^{-2k}\leqslant t$ then we split the integral into two parts
\begin{align*}
\left\|\int_{t-2^{-2k}}^{t}P_{t-s}\Delta_{k}f_{s}\mathrm{d}s\right\|_{L^{\infty}(\rho_{t})}&\lesssim 2^{-k\alpha}\int_{t-2^{-2k}}^{t}e^{-\mu(t-s)}\mathrm{d}s \left\|f\right\|_{\CC^{\alpha}(\rho_{t})}\lesssim 2^{-k(2+\alpha)}\|f\|_{L^{\infty}\CC^{\alpha}(\rho)}
\end{align*}
and
\begin{align*}
&\left\|\int_{0}^{t-2^{-2k}}P_{t-s}\Delta_{k}f_{s}\mathrm{d}s\right\|_{L^{\infty}(\rho_{t})}\lesssim e^{-\mu t} \int_{0}^{t-2^{-2k}}e^{\mu s} (t-s)^{-1-\varepsilon}\dd s \, 2^{-k(\alpha+2(1+\varepsilon))}\|f\|_{L^{\infty}\CC^{\alpha}(\rho)}\\
&\qquad\lesssim t^{-\varepsilon}e^{-\mu t} \int_{0}^{1-2^{-2k}/t}e^{\mu ts} (1-s)^{-1-\varepsilon}\dd s \, 2^{-k(\alpha+2(1+\varepsilon))}\|f\|_{L^{\infty}\CC^{\alpha}(\rho)}\\
&\qquad\lesssim 2^{-k(\alpha+2)}\|f\|_{L^{\infty}\CC^{\alpha}(\rho)}= 2^{-k(\alpha+2)}\|f\|_{L^{\infty}\CC^{\alpha}(\rho)}.
\end{align*}
Note that all the above inequalities are uniform over $t\geq 0$. Hence the first bound in \eqref{eq:1vw} follows. The second one is obtained as (recall that the weight is nonincreasing in time)
$$
\|w(t)\|_{\CC^{2+\alpha}(\rho_{t})}\lesssim e^{-\mu t}\|w_{0}\|_{\CC^{2+\alpha}(\rho_{t})} \lesssim \|w_{0}\|_{\CC^{2+\alpha}(\rho_{0})}.
$$

The time regularity of $w$ follows from
\begin{align*}
\|w(t)-w(s)\|_{L^{\infty}(\rho_{t})}&=\|P_{s}(P_{t-s}-\mathrm{Id})w_{0}\|_{L^{\infty}(\rho_{t})}\lesssim |t-s|^{(2+\alpha)/2}\|w_{0}\|_{\CC^{2+\alpha}(\rho_{0})}
\end{align*}
and due to
$$
v(t)-v(s)=(P_{t-s}-\mathrm{Id})v({s})+\int_{s}^{t} P_{t-r}f({r})\dd r,\quad s<t,
$$
we obtain
\begin{align*}
\|v(t)-v(s)\|_{L^{\infty}(\rho_{t})}&\lesssim |t-s|^{(2+\alpha)/2}\|v({s})\|_{\CC^{2+\alpha}(\rho_{t})}+ |t-s|^{(2+\alpha)/2}\|f \|_{L^{\infty}\CC^{\alpha}(\rho)}\\
&\lesssim |t-s|^{(2+\alpha)/2}\|f \|_{L^{\infty}\CC^{\alpha}(\rho)}.
\end{align*}
The proof is complete.
\end{proof}

Next, we derive a Schauder estimate for  parabolic equations including a cubic nonlinearity.

\begin{lemma}
  \label{lemma:schauder-par}
Let $\mu>0$ and let $\rho$ be a space-time weight.  Fix $\kappa > 0$ and let $\psi \in C\CC^{2 +
  \kappa} (\rho^{3 + \kappa}) \cap C^{1}L^{\infty} (\rho^{3+\kappa})\cap L^{\infty}L^{\infty}(\rho)$ be a classical solution
  to
  \[\LL  \psi + \psi^3 = \Psi ,\qquad \psi(0)=\psi_{0}.\]
Then
  \[ \| \psi \|_{C\CC^{2 + \kappa} (\rho^{3 + \kappa})} +\|\psi\|_{C^{1}L^{\infty}(\rho^{3+\kappa})} \lesssim \|\psi_{0}\|_{\CC^{2+\kappa}(\rho_{0})}+  \|
     \Psi \|_{C\CC^{\kappa} (\rho^{3 + \kappa})} + \| \psi \|_{L^{\infty}L^{\infty}
     (\rho)}^{3 + \kappa}. \]
\end{lemma}

\begin{proof}
Due to Lemma \ref{lemma:sch} it holds
$$
\|\psi\|_{C\CC^{2+\kappa}(\rho^{3+\kappa})}+\|\psi\|_{C^{1}L^{\infty}(\rho^{3+\kappa})}\lesssim\|\psi_{0}\|_{\CC^{2+\kappa}(\rho_{0})}+  \|\Psi\|_{C\CC^\kappa(\rho^{3+\kappa})}+\|\psi^3\|_{C\CC^\kappa(\rho^{3+\kappa})}
$$
 and we estimate  pointwise in time using Lemma \ref{lemma:interp2} and the weighted Young inequality
  \[ \| \psi^3 \|_{C\CC^{\kappa} (\rho^{3 + \kappa})} \lesssim \| \psi
     \|_{L^{\infty}L^{\infty} (\rho)}^2 \| \psi \|_{C\CC^{\kappa} (\rho^{1 + \kappa})}\lesssim \| \psi
     \|_{L^{\infty}L^{\infty} (\rho)}^2 \| \psi \|_{L^{\infty}L^{\infty} (\rho)}^{1 - \kappa / (2
     + \kappa)} \| \psi \|_{C\CC^{2 + \kappa} (\rho^{3 + \kappa})}^{\kappa / (2
     + \kappa)}
    \]
    $$
    \leq
    c\left( \|\psi\|_{L^{\infty}L^{\infty} (\rho)}^2 \| \psi \|_{L^{\infty}L^{\infty} (\rho)}^{1 - \kappa / (2
     + \kappa)}\right)^{(2 + \kappa) / 2} + \frac{1}{2}\left( \| \psi \|_{C\CC^{2 +
     \kappa} (\rho^{3 + \kappa})}^{\kappa / (2 + \kappa)} \right)^{(2 +
     \kappa) / \kappa}
    $$
    $$ \leqslant c\| \psi
     \|_{L^{\infty} L^{\infty}(\rho)}^{3 + \kappa} +
     \frac{1}{2} \| \psi \|_{C\CC^{2 + \kappa} (\rho^{3 + \kappa})} .
    $$
    Hence
    $$
     \| \psi^3 \|_{C\CC^{\kappa} (\rho^{3 + \kappa})} \leqslant c\| \psi
     \|_{L^{\infty}L^{\infty} (\rho)}^{3 + \kappa} +
     \frac{1}{2} \| \psi \|_{C\CC^{2 + \kappa} (\rho^{3 + \kappa})} 
    $$
and the claim follows.
\end{proof}

\subsection{Parabolic coercive estimates}

Similarly to Section \ref{s:ell2} we obtain the following maximum principle for parabolic equations.

\begin{lemma}
  \label{lemma:apriori-parabolic}
 Let $\mu\in\R$ and let $\rho$ be a  space-time weight. Fix $\kappa > 0$ and let $\psi \in C\CC^{2+\kappa}(\rho^{3+\kappa})\cap C^{1}L^{\infty}(\rho^{3+\kappa})\cap L^{\infty}L^{\infty}(\rho)$ be a classical
  solution to
  \[\LL  \psi + \psi^3 = \Psi,\qquad \psi(0)=\psi_{0} .\]
 Then the following a priori estimate holds
  \[ \| \psi  \|_{L^{\infty}L^{\infty} (\rho)} \lesssim 1+ \|\psi_{0}\|_{L^{\infty}(\rho_{0})}+ \|\Psi\|^{1/3}_{L^{\infty}L^{\infty}(\rho^{3})},
  \]
  where $\rho_{0}=\rho(0,\cdot)$.
\end{lemma}

\begin{proof}
Let $\bar\psi=\psi\rho$ and assume for the moment that $\bar\psi$ attains its (global) maximum $M=\bar\psi(t^{*},x^{*})$ at the point $(t^{*},x^{*})$. If $M\leq0$, then it is necessary to investigate the minimum point (or alternatively the maximum of $-\bar\psi$), which we discuss below. Let us therefore assume  that $M>0$. If $t_{*}=0$ then
$$
\bar\psi\leqslant \|\psi_{0}\|_{L^{\infty}(\rho_{0})}
$$
Assume that $t^{*}>0$.
Then
$$
\rho^{2}\partial_{t}\bar\psi+\rho^{2}(-\Delta+\mu)\bar\psi +\bar\psi^{3} = \rho^{3}\Psi +\rho\partial_{t}\rho\bar\psi - \rho^{2}(\Delta \rho) \psi - 2\rho^{2}\nabla\rho\nabla\psi.
$$
and
$$
\partial_{t}\bar\psi(t^{*},x^{*})= 0,\qquad\nabla\bar\psi(t^{*},x^{*})=0,\qquad\Delta\bar\psi(t^{*},x^{*})\leq0
$$
hence $\rho\nabla\psi=-\psi\nabla\rho$. Consequently $-\rho^{2}\Delta\bar\psi(t^{*},x^{*})\geq0$ and also $\rho\partial_{t}\rho\bar\psi(t^{*},x^{*})\leqslant 0$ since $\partial_{t}\rho
\leqslant 0$. Hence
$$
M^{3}\leqslant \left[\rho^{3}\Psi -\mu\rho^{2}\bar\psi - \rho^{2}(\Delta \rho) \psi - 2\rho^{2}\nabla\rho\nabla\psi\right]_{|_{(t^{*},x^{*})}}
$$
$$
\leqslant \|\Psi\|_{L^{\infty}L^{\infty}(\rho^{3})} +\rho^{2}(t^{*},x^{*})\left[|\mu|+\left\|\frac{\nabla\rho}{\rho}\right\|^{2}_{L^{\infty}}+\left\|\frac{\Delta\rho}{\rho}\right\|_{L^{\infty}}\right]\|\bar\psi\|_{L^{\infty}L^{\infty}}
$$
$$
\leqslant \|\Psi\|_{L^{\infty}L^{\infty}(\rho^{3})} +c_{\rho,\mu}\|\bar\psi\|_{L^{\infty}L^{\infty}}.
$$
Therefore we deduce that
$$
\bar\psi\lesssim \|\psi_{0}\|_{L^{\infty}(\rho_{0})}+\|\Psi\|^{1/3}_{L^{\infty}L^{\infty}(\rho^{3})} +c_{\rho,\mu}\|\bar\psi\|^{1/3}_{L^{\infty}L^{\infty}}.
$$
The same argument applied to $-\bar\psi$ yields
$$
-\bar\psi\lesssim \|\psi_{0}\|_{L^{\infty}(\rho_{0})}+\|\Psi\|^{1/3}_{L^{\infty}L^{\infty}(\rho^{3})} +c_{\rho,\mu}\|\bar\psi\|^{1/3}_{L^{\infty}L^{\infty}}
$$
hence, taking supremum over $(t,x)\in [0,\8)\times\R^{d}$ and applying the weighted Young inequality yields the claim.

Next, we consider the situation when $\psi\rho$ does not attain its global maximum. Since $\psi\rho$ is smooth and  bounded on $[0,\infty)\times\R^{d}$ due to the assumption, it follows that  $\psi\rho^{1+\delta}$ vanishes at infinity for every $\delta\in (0,1)$. Consequently, it has a global maximum point and the previous part of the proof applies with $\rho$ replaced by $\rho^{\delta}$. The conclusion then follows by sending $\delta\to 0$ since the corresponding constant $c_{\rho^{1+\delta},\mu}$ is bounded uniformly in $\delta\in (0,1).$
\end{proof}

\subsection{Paracontrolled calculus}
\label{ssec:para}

The foundations of paracontrolled calculus were laid down in the seminal work \cite{GIP} of Gubinelli, Imkeller and Perkowski, to which we  refer the reader for a number of facts used here. We refer to the book \cite{BCD} of Bahouri, Chemin and Danchin for a gentle introduction to the use of paradifferential calculus in the study of nonlinear PDEs. We shall then freely use the decomposition  $fg=f\prec g+f\circ g+f\succ g$, where $f\prec g=g\succ f$ and $f\circ g$, respectively, stands for the paraproduct of $f$ by $g$ and the corresponding resonant term, defined in terms of Littlewood--Paley decomposition.

The following basic results are obtained similarly to the unweighted setting.

\begin{lemma}\label{lem:bcd}
Let $\rho$ be an admissible weight.
\begin{enumerate}
\item Let $\mathcal{A}$ be an annulus, let $\alpha\in\R$ and let $(u_j)_{j\geq-1}$ be a sequence of smooth functions such that $\mathcal{F}u_j$ is supported in $2^j\mathcal{A}$ and  $\|u_j\|_{L^\infty(\rho)}\lesssim 2^{-j\alpha}$ for all $j\geq-1$. Then
$$
u=\sum_{j\geq-1}u_j\in \CC^\alpha(\rho)\qquad\text{and}\qquad\|u\|_{\CC^\alpha(\rho)}\lesssim \sup_{j\geq -1}\{2^{j\alpha}\|u_j\|_{L^\infty(\rho)}\}.
$$
\item Let $\mathcal{B}$ be a ball, let $\alpha>0$ and let $(u_j)_{j\geq-1}$ be a sequence of smooth functions such that $\mathcal{F}u_j$ is supported in $2^j\mathcal{B}$ and  $\|u_j\|_{L^\infty(\rho)}\lesssim 2^{-j\alpha}$ for all $j\geq-1$. Then
$$
u=\sum_{j\geq-1}u_j\in \CC^\alpha(\rho)\qquad\text{and}\qquad\|u\|_{\CC^\alpha(\rho)}\lesssim \sup_{j\geq -1}\{2^{j\alpha}\|u_j\|_{L^\infty(\rho)}\}.
$$
\end{enumerate}
\end{lemma}

\begin{proof}
The proof follows the lines of \cite[Lemma A.3]{GIP}.
\end{proof}

\begin{lemma}[Paraproduct estimates]\label{lem:para}
Let $\rho_{1},\rho_{2}$ be admissible weights  and $\beta\in\R$. Then  it holds
\begin{equation*}
\|f\prec g\|_{\CC^\beta(\rho_{1}\rho_{2})}\lesssim_\beta\|f\|_{L^\infty(\rho_{1})}\|g\|_{\CC^{\beta}(\rho_{2})},
\end{equation*}
and if $\alpha<0$ then
\begin{equation*}
\|f\prec g\|_{\CC^{\alpha+\beta}(\rho_{1}\rho_{2})}\lesssim_{\alpha,\beta}\|f\|_{\CC^{\alpha}(\rho_{1})}\|g\|_{\CC^{\beta}(\rho_{2})}.
\end{equation*}
If  $\alpha+\beta>0$ then it holds
\begin{equation*}
\|f\circ g\|_{\CC^{\alpha+\beta}(\rho_{1}\rho_{2})}\lesssim_{\alpha,\beta}\|f\|_{\CC^{\alpha}(\rho_{1})}\|g\|_{\CC^{\beta}(\rho_{2})}.
\end{equation*}
\end{lemma}

\begin{proof}
The proof follows the lines of \cite[Lemma 2.1]{GIP} and uses  Lemma \ref{lem:bcd} instead of \cite[Lemma A.3]{GIP}
\end{proof}

We also obtain the following weighted analog of \cite[Lemma 2.2, Lemma 2.3]{GIP}, which is proved analogously.

\begin{lemma}\label{lem:2.2}
Let $\rho_{1},\rho_{2}$ be admissible weights   and $\alpha\in(0,1)$ and $\beta\in\R$. For all $j\geq -1$ it holds
$$
\|\Delta_j(fg)-f\Delta_jg\|_{L^\infty(\rho_{1}\rho_{2})}\lesssim 2^{-j\alpha}\|f\|_{\CC^{\alpha}(\rho_{1})}\|g\|_{L^\infty(\rho_{2})}.
$$
$$
\|\Delta_j(f\prec g)-f\Delta_j g\|_{L^\infty(\rho_{1}\rho_{2})}\lesssim 2^{-j(\alpha+\beta)}\|f\|_{\CC^\alpha(\rho_{1})}\|g\|_{\CC^\beta(\rho_{2})}.
$$
\end{lemma}

With this in hand, we derive a weighted commutator estimate.

\begin{lemma}[Commutator lemma]\label{lem:com}
Let $\rho_{1}, \rho_{2}, \rho_{3}$ be admissible weights and let $\alpha\in (0,1)$ and $\beta,\gamma\in \R$ such that $\alpha+\beta+\gamma>0$ and $\beta+\gamma<0$. Then there exist a trilinear bounded operator $\mathrm{com}$ satisfying
$$
\|\mathrm{com}(f,g,h)\|_{\CC^{\alpha+\beta+\gamma}(\rho_{1}\rho_{2}\rho_{3})}\lesssim \|f\|_{\CC^\alpha(\rho_{1})}\|g\|_{\CC^\beta(\rho_{2})}\|h\|_{\CC^\gamma(\rho_{3})}
$$
and for smooth functions $f,g,h$
$$
\mathrm{com}(f,g,h)=(f\prec g)\circ h - f(g\circ h).
$$
\end{lemma}

\begin{proof}
In view of Lemma \ref{lem:bcd}, Lemma \ref{lem:para}, Lemma \ref{lem:2.2}, the proof is the same as the proof of \cite[Lemma 2.4]{GIP}.
\end{proof}

%
%

Moreover, we will make use of the time-mollified paraproducts as introduced in \cite[Section 5]{GIP}. Let $Q:\R\to\R_{+}$ be a smooth function, supported in $[-1,1]$ and $\int_{\R}Q(s)\mathrm{d}s=1$, and for $i\geq -1$ define the operator $Q_{i}:C\CC^{\alpha}(\rho)\to C\CC^{\alpha}(\rho)$ by
$$
Q_{i}f(t)=\int_{\R}2^{2i}Q(2^{2i}(t-s))f(s\vee 0)\mathrm{d} s.
$$
Finally, we define the modified paraproduct of $f,g\in C\CC^{\alpha}(\rho)$ by
$$
f\Prec g := \sum_{i\geq -1}(S_{i-1}Q_{i}f)\Delta_{i} g.
$$
Setting  $\LL=\partial_{t} +(-\Delta+\mu)$, the following useful properties of this paraproduct in weighted Besov spaces  can be shown similarly to \cite[Lemma 5.1]{GIP}. Here we denote by $[\LL,f\Prec]$ the commutator between $\LL$ and $f\Prec$, that is, $[\LL,f\Prec]g=\LL (f\Prec g) -f\Prec(\LL g)$.

\begin{lemma}\label{lem:5.1}
Let $\rho_{1},\rho_{2}$ be admissible space-time  weights. Let $\alpha\in (0,1),$ $\beta\in \R$, and let $f\in C\CC^{\alpha}(\rho_{1})\cap C^{\alpha/2}L^{\infty}(\rho_{1})$ and $g\in C\CC^{\beta}(\rho_{2})$. Then
$$
\big\|[\LL,f\Prec] g\big\|_{C\CC^{\alpha+\beta-2}(\rho_{1}\rho_{2})}\lesssim \big( \|f\|_{C^{\alpha/2}L^{\infty}(\rho_{1})}+\|f\|_{C\CC^{\alpha}(\rho_{1})} \big)\|g\|_{C\CC^{\beta}(\rho_{2})},
$$
and
$$
\|f\prec g-f\Prec g\|_{C\CC^{\alpha+\beta}(\rho_{1}\rho_{2})}\lesssim \|f\|_{C^{\alpha/2}L^{\infty}(\rho_{1})}\|g\|_{C\CC^{\beta}(\rho_{2})}.
$$
\end{lemma}

\section{Probabilistic analysis}
\label{s:proba}

\subsection{Space white noise}
\label{ssec:renorm-el}

Let $\xi$ be a space white noise on $\R^{d}$, that is, a family of centered Gaussian random variables $\{\xi(h);\,h\in L^{2}(\R^{d}) \}$ such that
$$
\mathbb{E}[\xi(h)^{2}]=\|h\|^{2}_{L^{2}}.
$$
Let $\xi_{M}$ denote its periodization on $\mathbb{T}^{d}_{M}=(M\mathbb{T})^{d}=\left[-\frac{M}{2},\frac{M}{2}\right]^{d}$ given by
$$
\xi_{M}(h):=\xi(h_{M}),\qquad\text{where }\ h_{M}(x)=\mathbf{1}_{[-\frac{M}{2},\frac{M}{2}]^{d}}(x)\sum_{y\in M\mathbb{Z}^{d}}h(x+y).
$$
Let 
\begin{equation*}
 \Q X = \xi, \qquad  \Q X_{M} = \xi_{M},
\end{equation*}
and denote by $  \llbracket X^2 \rrbracket, \llbracket X^3 \rrbracket $ and $  \llbracket X_{M}^2 \rrbracket, \llbracket X_{M}^3 \rrbracket $ the corresponding Wick powers. They can be constructed by using a suitable mollification $\xi_{\varepsilon}=\xi*\eta_{\varepsilon}$ and $\xi_{M,\varepsilon}=\xi_{M}*\eta_{\varepsilon}$ (where $\eta_{\varepsilon}$ stands for a smoothing kernel) and setting
$$
 \Q X_{\varepsilon} = \xi_{\varepsilon}, \qquad  \Q X_{M,\varepsilon} = \xi_{M,\varepsilon},
$$
$$
 \llbracket X^2 \rrbracket  \assign \lim_{\varepsilon\to 0}  \llbracket X_{\varepsilon}^2 \rrbracket  \assign \lim_{\varepsilon\to 0}  X_{\varepsilon}^2 - a_{\varepsilon}, \qquad   \llbracket X^3 \rrbracket\assign\lim_{\varepsilon\to 0}  \llbracket X_{\varepsilon}^3 \rrbracket\assign \lim_{\varepsilon\to 0} X_{\varepsilon}^3 - 3 a_{\varepsilon} X_{\varepsilon},
 $$
 $$
 \llbracket X_{M}^2 \rrbracket  \assign \lim_{\varepsilon\to 0}  \llbracket X_{M,\varepsilon}^2 \rrbracket  \assign \lim_{\varepsilon\to 0}  X_{M,\varepsilon}^2 - a_{M,\varepsilon}, \qquad   \llbracket X_{M}^3 \rrbracket\assign\lim_{\varepsilon\to 0} \llbracket X_{M,\varepsilon}^3 \rrbracket\assign\lim_{\varepsilon\to 0} X_{M,\varepsilon}^3 - 3 a_{M,\varepsilon} X_{M,\varepsilon},
 $$
 where $a_{\varepsilon}= \mathbbm{E}[X_{\varepsilon}^2(0)]$  and $a_{M,\varepsilon}= \mathbbm{E}[X_{M,\varepsilon}^2(0)]$ are  constants diverging  as $\varepsilon\to0$ and the limits are understood in a suitable Besov space a.s. More precisely, the following result holds.

 \begin{theorem}\label{thm:renorm}
Let $d=4$. Let $\rho(x)=\langle x\rangle^{-\nu}$ for some $\nu>0$. Then there exist random distributions $X,\llbracket X^{2}\rrbracket,\llbracket X^{3}\rrbracket$ and $X_{M},\llbracket X_{M}^{2}\rrbracket,\llbracket X_{M}^{3}\rrbracket$ given by the formulas above, such that for every $\kappa,\sigma>0$  it holds 
$$
\|X\|_{\CC^{-\kappa}(\rho^{\sigma})},\|\llbracket X^{2}\rrbracket\|_{\CC^{-\kappa}(\rho^{\sigma})}
,\|\llbracket X^{3}\rrbracket\|_{\CC^{-\kappa}(\rho^{\sigma})}\lesssim 1,
$$
$$
\|X_{M}\|_{\CC^{-\kappa}(\mathbb{T}^{4}_{M})},\|\llbracket X_{M}^{2}\rrbracket\|_{\CC^{-\kappa}(\mathbb{T}^{4}_{M})}
,\|\llbracket X_{M}^{3}\rrbracket\|_{\CC^{-\kappa}(\mathbb{T}^{4}_{M})}\lesssim 1,
$$
and in addition $X_{M}\to X$, $\llbracket X_{M}^{2}\rrbracket\to \llbracket X^{2}\rrbracket$, $\llbracket X_{M}^{3}\rrbracket\to \llbracket X^{3}\rrbracket$ in $\CC^{-\kappa}(\rho^{\sigma})$ a.s. as $M\to\infty$. \end{theorem}
 
 \begin{proof} We give a sketch of the proof since similar arguments are already present in the literature on parabolic $\Phi^4_d$ models, and in particular in the work of Mourrat and Weber~\cite{MW17}. Following the approach of Gubinelli and Perkowski~\cite{gubinelli_kpz_2017} we represent the random fields $X$ and $X_M$ as Wiener integrals over a white noise $W$ on $\R^4$. As a consequence we can write 
 \[ X (x) = \int_{\mathbbm{R}^4} e^{2 \pi i \theta \cdot x} \frac{W (\mathd
   \theta)}{\mu + | \theta |^2}, \qquad X_M (x) = \int_{\mathbbm{R}^4} e^{2
   \pi i [\theta]_M \cdot x} \frac{W (\mathd \theta)}{\mu + | [\theta]_M |^2},
\]
where $([\theta]_M)^i = M^{- 1} \lfloor M \theta^i - 1 / 2 \rfloor$, $i = 1,
\ldots, d$ is the discretization of $\theta\in\R^4$ on a grid of size $M^{-1}$. The reader can check that this gives a periodic random field with the correct covariance. Wick powers of $X$ (or $X_M$) can then be expressed as multiple Wiener integrals over $W$. We present the details for $\llbracket X^3 \rrbracket $:
\[ \llbracket X^3 \rrbracket (x) = \int_{(\mathbbm{R}^4)^3} e^{2 \pi i
   (\theta_1 + \theta_2 + \theta_3) \cdot x} \frac{W (\mathd \theta_1 \mathd
   \theta_2 \mathd \theta_3)}{\prod_{i = 1}^3 (\mu + | \theta_i |^2)}, \]
\[ \llbracket X^3_M \rrbracket (x) = \int_{(\mathbbm{R}^4)^3} e^{2 \pi i
   ([\theta_1]_M + [\theta_2]_M + [\theta_3]_M) \cdot x} \frac{W (\mathd
   \theta_1 \mathd \theta_2 \mathd \theta_3)}{\prod_{i = 1}^3 (\mu + |
   [\theta_i]_M |^2)} . \]
And $L^2$ bound on the Littlewood--Paley block of these quantities reads, for $k\geq -1$,
\[ \mathbbm{E} [| \Delta_k \llbracket X^3_M \rrbracket (x) |^2] =
   \int_{(\mathbbm{R}^4)^3} K_k ([\theta_1]_M + [\theta_2]_M +
   [\theta_3]_M)^2 \frac{\mathd \theta_1 \mathd \theta_2 \mathd
   \theta_3}{\prod_{i = 1}^3 (\mu + | [\theta_i]_M |^2)^2} \]
\[ \lesssim \int_{(\mathbbm{R}^4)^3} K_k (\theta_1 + \theta_2 + \theta_3)^2
   \frac{\mathd \theta_1 \mathd \theta_2 \mathd \theta_3}{\prod_{i = 1}^3 (\mu
   + | \theta_i |^2)^2} \lesssim 1, \]
where $K_{k}$ is the Fourier multiplier associated with $\Delta_k$. 
From this we deduce by hypercontractivity that $\mathbbm{E} [| \Delta_k
\llbracket X^3_M \rrbracket (x) |^p] \lesssim 1$ and therefore that
\[ \mathbbm{E} [\| \Delta_k \llbracket X^3_M \rrbracket \|_{L^p
   (\rho^{\sigma})}^p] \lesssim \int_{\mathbbm{R}^4} \mathbbm{E} [| \Delta_k
   \llbracket X^3_M \rrbracket (x) |^p] \rho (x)^{\sigma p} \mathd x \lesssim
   \int_{\mathbbm{R}^4} \rho (x)^{\sigma p} \mathd x \lesssim 1, \]
for $p$ sufficiently large so that the space integral is finite. As a
consequence of Bernstein inequality it follows that
\[ \| \Delta_k \llbracket X^3_M \rrbracket \|_{L^{\infty} (\rho^{\sigma})}
   \lesssim 2^{4 k / p} \| \Delta_k \llbracket X^3_M \rrbracket \|_{L^p
   (\rho^{\sigma})} \]
and therefore 
\[ \mathbbm{E} (\|  \llbracket X^3_M \rrbracket \|_{C^{- \kappa}
   (\rho^{\sigma})}^p) < \infty, \]
for $p$ large enough and $\kappa > 0$ small. Convergence of $\llbracket X^3_M
\rrbracket$ to $\llbracket X^3 \rrbracket$ can be handled by coupling,
observing that estimation of $\llbracket X^3_M \rrbracket - \llbracket X^3
\rrbracket$ involves computations similar to the above. Indeed, it holds
\[ \mathbbm{E} [| \Delta_k (\llbracket X^3_M \rrbracket - \llbracket X^3
   \rrbracket) (x) |^2] = \int_{(\mathbbm{R}^4)^3} \left( \frac{K_k
   ([\theta_1]_M + [\theta_2]_M + [\theta_3]_M)}{\prod_{i = 1}^3 (\mu + |
   [\theta_i]_M |^2)} - \frac{K_k (\theta_1 + \theta_2 +
   \theta_3)}{\prod_{i = 1}^3 (\mu + | \theta_i |^2)} \right)^2 \mathd
   \theta_1 \mathd \theta_2 \mathd \theta_3 \]
which by dominated convergence tends to zero as $M \rightarrow \infty$.
Therefore we can estimate
\[ \mathbbm{E} [\| \Delta_k (\llbracket X^3_M \rrbracket - \llbracket X^3
   \rrbracket) \|_{L^p (\rho^{\sigma})}^p] \lesssim \int_{\mathbbm{R}^4}
   \mathbbm{E} [| \Delta_k (\llbracket X^3_M \rrbracket - \llbracket X^3
   \rrbracket) (x) |^p] \rho (x)^{\sigma p} \mathd x \lesssim o_M (1) . \]
 \end{proof}
 
 This result will be used for the study of elliptic $\Phi^{4}$ model in dimension 4, see Section \ref{sec:44}. When $d=5$ then the space white noise becomes more irregular and our analysis requires additional probabilistic objects. More precisely, we let
\[ \Q X^{\!\resizebox{0.6em}{!}{
\begin{tikzpicture}
\pgfpathmoveto{\pgfqpoint{0cm}{-0.035cm}}
\pgfpathlineto{\pgfqpoint{1.376cm}{-0.035cm}}
\pgfpathlineto{\pgfqpoint{1.376cm}{1.552cm}}
\pgfpathlineto{\pgfqpoint{0cm}{1.552cm}}
\pgfpathclose
\pgfusepath{clip}
\begin{pgfscope}
\begin{pgfscope}
\pgfpathmoveto{\pgfqpoint{0cm}{-0.035cm}}
\pgfpathlineto{\pgfqpoint{1.376cm}{-0.035cm}}
\pgfpathlineto{\pgfqpoint{1.376cm}{1.552cm}}
\pgfpathlineto{\pgfqpoint{0cm}{1.552cm}}
\pgfpathclose
\pgfusepath{clip}
\begin{pgfscope}
\begin{pgfscope}
\pgfsetdash{}{0cm}
\pgfsetlinewidth{0.818mm}
\pgfsetroundcap
\pgfsetroundjoin
\pgfsetmiterlimit{7.0}
\definecolor{eps2pgf_color}{gray}{0}\pgfsetstrokecolor{eps2pgf_color}\pgfsetfillcolor{eps2pgf_color}
\pgfpathmoveto{\pgfqpoint{0.117cm}{1.421cm}}
\pgfpathlineto{\pgfqpoint{0.682cm}{0.671cm}}
\pgfpathlineto{\pgfqpoint{1.246cm}{1.421cm}}
\pgfusepath{stroke}
\end{pgfscope}
\definecolor{eps2pgf_color}{gray}{0}\pgfsetstrokecolor{eps2pgf_color}\pgfsetfillcolor{eps2pgf_color}
\pgfpathmoveto{\pgfqpoint{0.273cm}{1.395cm}}
\pgfpathcurveto{\pgfqpoint{0.273cm}{1.432cm}}{\pgfqpoint{0.259cm}{1.467cm}}{\pgfqpoint{0.233cm}{1.492cm}}
\pgfpathcurveto{\pgfqpoint{0.207cm}{1.518cm}}{\pgfqpoint{0.173cm}{1.532cm}}{\pgfqpoint{0.137cm}{1.532cm}}
\pgfpathcurveto{\pgfqpoint{0.1cm}{1.532cm}}{\pgfqpoint{0.066cm}{1.518cm}}{\pgfqpoint{0.04cm}{1.492cm}}
\pgfpathcurveto{\pgfqpoint{0.014cm}{1.467cm}}{\pgfqpoint{0cm}{1.432cm}}{\pgfqpoint{0cm}{1.395cm}}
\pgfpathcurveto{\pgfqpoint{0cm}{1.359cm}}{\pgfqpoint{0.014cm}{1.324cm}}{\pgfqpoint{0.04cm}{1.299cm}}
\pgfpathcurveto{\pgfqpoint{0.066cm}{1.273cm}}{\pgfqpoint{0.1cm}{1.258cm}}{\pgfqpoint{0.137cm}{1.258cm}}
\pgfpathcurveto{\pgfqpoint{0.173cm}{1.258cm}}{\pgfqpoint{0.207cm}{1.273cm}}{\pgfqpoint{0.233cm}{1.299cm}}
\pgfpathcurveto{\pgfqpoint{0.259cm}{1.324cm}}{\pgfqpoint{0.273cm}{1.359cm}}{\pgfqpoint{0.273cm}{1.395cm}}
\pgfusepath{fill}
\begin{pgfscope}
\pgfsetdash{}{0cm}
\pgfsetlinewidth{0.818mm}
\pgfsetmiterlimit{7.0}
\pgfpathmoveto{\pgfqpoint{0.682cm}{0.671cm}}
\pgfpathlineto{\pgfqpoint{0.679cm}{1.418cm}}
\pgfusepath{stroke}
\end{pgfscope}
\pgfpathmoveto{\pgfqpoint{0.815cm}{1.399cm}}
\pgfpathcurveto{\pgfqpoint{0.815cm}{1.435cm}}{\pgfqpoint{0.801cm}{1.47cm}}{\pgfqpoint{0.775cm}{1.496cm}}
\pgfpathcurveto{\pgfqpoint{0.75cm}{1.521cm}}{\pgfqpoint{0.715cm}{1.536cm}}{\pgfqpoint{0.679cm}{1.536cm}}
\pgfpathcurveto{\pgfqpoint{0.643cm}{1.536cm}}{\pgfqpoint{0.608cm}{1.521cm}}{\pgfqpoint{0.582cm}{1.496cm}}
\pgfpathcurveto{\pgfqpoint{0.557cm}{1.47cm}}{\pgfqpoint{0.542cm}{1.435cm}}{\pgfqpoint{0.542cm}{1.399cm}}
\pgfpathcurveto{\pgfqpoint{0.542cm}{1.363cm}}{\pgfqpoint{0.557cm}{1.328cm}}{\pgfqpoint{0.582cm}{1.302cm}}
\pgfpathcurveto{\pgfqpoint{0.608cm}{1.276cm}}{\pgfqpoint{0.643cm}{1.262cm}}{\pgfqpoint{0.679cm}{1.262cm}}
\pgfpathcurveto{\pgfqpoint{0.715cm}{1.262cm}}{\pgfqpoint{0.75cm}{1.276cm}}{\pgfqpoint{0.775cm}{1.302cm}}
\pgfpathcurveto{\pgfqpoint{0.801cm}{1.328cm}}{\pgfqpoint{0.815cm}{1.363cm}}{\pgfqpoint{0.815cm}{1.399cm}}
\pgfusepath{fill}
\pgfpathmoveto{\pgfqpoint{1.345cm}{1.371cm}}
\pgfpathcurveto{\pgfqpoint{1.345cm}{1.408cm}}{\pgfqpoint{1.331cm}{1.442cm}}{\pgfqpoint{1.305cm}{1.468cm}}
\pgfpathcurveto{\pgfqpoint{1.28cm}{1.494cm}}{\pgfqpoint{1.245cm}{1.508cm}}{\pgfqpoint{1.209cm}{1.508cm}}
\pgfpathcurveto{\pgfqpoint{1.172cm}{1.508cm}}{\pgfqpoint{1.138cm}{1.494cm}}{\pgfqpoint{1.112cm}{1.468cm}}
\pgfpathcurveto{\pgfqpoint{1.087cm}{1.442cm}}{\pgfqpoint{1.072cm}{1.408cm}}{\pgfqpoint{1.072cm}{1.371cm}}
\pgfpathcurveto{\pgfqpoint{1.072cm}{1.335cm}}{\pgfqpoint{1.087cm}{1.3cm}}{\pgfqpoint{1.112cm}{1.274cm}}
\pgfpathcurveto{\pgfqpoint{1.138cm}{1.249cm}}{\pgfqpoint{1.172cm}{1.234cm}}{\pgfqpoint{1.209cm}{1.234cm}}
\pgfpathcurveto{\pgfqpoint{1.245cm}{1.234cm}}{\pgfqpoint{1.28cm}{1.249cm}}{\pgfqpoint{1.305cm}{1.274cm}}
\pgfpathcurveto{\pgfqpoint{1.331cm}{1.3cm}}{\pgfqpoint{1.345cm}{1.335cm}}{\pgfqpoint{1.345cm}{1.371cm}}
\pgfusepath{fill}
\begin{pgfscope}
\pgfsetdash{}{0cm}
\pgfsetlinewidth{0.818mm}
\pgfsetroundcap
\pgfsetmiterlimit{4.0}
\pgfpathmoveto{\pgfqpoint{0.682cm}{0.671cm}}
\pgfpathlineto{\pgfqpoint{0.682cm}{0.042cm}}
\pgfusepath{stroke}
\end{pgfscope}
\end{pgfscope}
\end{pgfscope}
\end{pgfscope}
\end{tikzpicture}}} = \llbracket X^3 \rrbracket, \quad \Q
   X^{\!\resizebox{0.6em}{!}{
\begin{tikzpicture}
\pgfpathmoveto{\pgfqpoint{0cm}{0cm}}
\pgfpathlineto{\pgfqpoint{1.376cm}{0cm}}
\pgfpathlineto{\pgfqpoint{1.376cm}{1.588cm}}
\pgfpathlineto{\pgfqpoint{0cm}{1.588cm}}
\pgfpathclose
\pgfusepath{clip}
\begin{pgfscope}
\begin{pgfscope}
\pgfpathmoveto{\pgfqpoint{0cm}{0cm}}
\pgfpathlineto{\pgfqpoint{1.376cm}{0cm}}
\pgfpathlineto{\pgfqpoint{1.376cm}{1.588cm}}
\pgfpathlineto{\pgfqpoint{0cm}{1.588cm}}
\pgfpathclose
\pgfusepath{clip}
\begin{pgfscope}
\begin{pgfscope}
\definecolor{eps2pgf_color}{gray}{0.976471}\pgfsetstrokecolor{eps2pgf_color}\pgfsetfillcolor{eps2pgf_color}
\pgfpathmoveto{\pgfqpoint{0cm}{0cm}}
\pgfpathlineto{\pgfqpoint{1.376cm}{0cm}}
\pgfpathlineto{\pgfqpoint{1.376cm}{1.588cm}}
\pgfpathlineto{\pgfqpoint{0cm}{1.588cm}}
\pgfpathclose
\pgfusepath{fill}
\end{pgfscope}
\begin{pgfscope}
\pgfsetdash{}{0cm}
\pgfsetlinewidth{0.818mm}
\pgfsetroundcap
\pgfsetroundjoin
\pgfsetmiterlimit{7.0}
\definecolor{eps2pgf_color}{gray}{0}\pgfsetstrokecolor{eps2pgf_color}\pgfsetfillcolor{eps2pgf_color}
\pgfpathmoveto{\pgfqpoint{0.117cm}{1.476cm}}
\pgfpathlineto{\pgfqpoint{0.682cm}{0.726cm}}
\pgfpathlineto{\pgfqpoint{1.246cm}{1.476cm}}
\pgfusepath{stroke}
\end{pgfscope}
\definecolor{eps2pgf_color}{gray}{0}\pgfsetstrokecolor{eps2pgf_color}\pgfsetfillcolor{eps2pgf_color}
\pgfpathmoveto{\pgfqpoint{0.273cm}{1.451cm}}
\pgfpathcurveto{\pgfqpoint{0.273cm}{1.487cm}}{\pgfqpoint{0.259cm}{1.522cm}}{\pgfqpoint{0.233cm}{1.547cm}}
\pgfpathcurveto{\pgfqpoint{0.207cm}{1.573cm}}{\pgfqpoint{0.173cm}{1.588cm}}{\pgfqpoint{0.137cm}{1.588cm}}
\pgfpathcurveto{\pgfqpoint{0.1cm}{1.588cm}}{\pgfqpoint{0.066cm}{1.573cm}}{\pgfqpoint{0.04cm}{1.547cm}}
\pgfpathcurveto{\pgfqpoint{0.014cm}{1.522cm}}{\pgfqpoint{0cm}{1.487cm}}{\pgfqpoint{0cm}{1.451cm}}
\pgfpathcurveto{\pgfqpoint{0cm}{1.414cm}}{\pgfqpoint{0.014cm}{1.379cm}}{\pgfqpoint{0.04cm}{1.354cm}}
\pgfpathcurveto{\pgfqpoint{0.066cm}{1.328cm}}{\pgfqpoint{0.1cm}{1.314cm}}{\pgfqpoint{0.137cm}{1.314cm}}
\pgfpathcurveto{\pgfqpoint{0.173cm}{1.314cm}}{\pgfqpoint{0.207cm}{1.328cm}}{\pgfqpoint{0.233cm}{1.354cm}}
\pgfpathcurveto{\pgfqpoint{0.259cm}{1.379cm}}{\pgfqpoint{0.273cm}{1.414cm}}{\pgfqpoint{0.273cm}{1.451cm}}
\pgfusepath{fill}
\pgfpathmoveto{\pgfqpoint{1.345cm}{1.426cm}}
\pgfpathcurveto{\pgfqpoint{1.345cm}{1.463cm}}{\pgfqpoint{1.331cm}{1.497cm}}{\pgfqpoint{1.305cm}{1.523cm}}
\pgfpathcurveto{\pgfqpoint{1.28cm}{1.549cm}}{\pgfqpoint{1.245cm}{1.563cm}}{\pgfqpoint{1.209cm}{1.563cm}}
\pgfpathcurveto{\pgfqpoint{1.172cm}{1.563cm}}{\pgfqpoint{1.138cm}{1.549cm}}{\pgfqpoint{1.112cm}{1.523cm}}
\pgfpathcurveto{\pgfqpoint{1.087cm}{1.497cm}}{\pgfqpoint{1.072cm}{1.463cm}}{\pgfqpoint{1.072cm}{1.426cm}}
\pgfpathcurveto{\pgfqpoint{1.072cm}{1.39cm}}{\pgfqpoint{1.087cm}{1.355cm}}{\pgfqpoint{1.112cm}{1.329cm}}
\pgfpathcurveto{\pgfqpoint{1.138cm}{1.304cm}}{\pgfqpoint{1.172cm}{1.289cm}}{\pgfqpoint{1.209cm}{1.289cm}}
\pgfpathcurveto{\pgfqpoint{1.245cm}{1.289cm}}{\pgfqpoint{1.28cm}{1.304cm}}{\pgfqpoint{1.305cm}{1.329cm}}
\pgfpathcurveto{\pgfqpoint{1.331cm}{1.355cm}}{\pgfqpoint{1.345cm}{1.39cm}}{\pgfqpoint{1.345cm}{1.426cm}}
\pgfusepath{fill}
\begin{pgfscope}
\pgfsetdash{}{0cm}
\pgfsetlinewidth{0.818mm}
\pgfsetroundcap
\pgfsetmiterlimit{4.0}
\pgfpathmoveto{\pgfqpoint{0.682cm}{0.726cm}}
\pgfpathlineto{\pgfqpoint{0.682cm}{0.097cm}}
\pgfusepath{stroke}
\end{pgfscope}
\end{pgfscope}
\end{pgfscope}
\end{pgfscope}
\end{tikzpicture}}} = \llbracket X^2 \rrbracket, \]
   \[ \Q X^{\!\resizebox{0.6em}{!}{
\begin{tikzpicture}
\pgfpathmoveto{\pgfqpoint{0cm}{-0.035cm}}
\pgfpathlineto{\pgfqpoint{1.376cm}{-0.035cm}}
\pgfpathlineto{\pgfqpoint{1.376cm}{1.552cm}}
\pgfpathlineto{\pgfqpoint{0cm}{1.552cm}}
\pgfpathclose
\pgfusepath{clip}
\begin{pgfscope}
\begin{pgfscope}
\pgfpathmoveto{\pgfqpoint{0cm}{-0.035cm}}
\pgfpathlineto{\pgfqpoint{1.376cm}{-0.035cm}}
\pgfpathlineto{\pgfqpoint{1.376cm}{1.552cm}}
\pgfpathlineto{\pgfqpoint{0cm}{1.552cm}}
\pgfpathclose
\pgfusepath{clip}
\begin{pgfscope}
\begin{pgfscope}
\pgfsetdash{}{0cm}
\pgfsetlinewidth{0.818mm}
\pgfsetroundcap
\pgfsetroundjoin
\pgfsetmiterlimit{7.0}
\definecolor{eps2pgf_color}{gray}{0}\pgfsetstrokecolor{eps2pgf_color}\pgfsetfillcolor{eps2pgf_color}
\pgfpathmoveto{\pgfqpoint{0.117cm}{1.421cm}}
\pgfpathlineto{\pgfqpoint{0.682cm}{0.671cm}}
\pgfpathlineto{\pgfqpoint{1.246cm}{1.421cm}}
\pgfusepath{stroke}
\end{pgfscope}
\definecolor{eps2pgf_color}{gray}{0}\pgfsetstrokecolor{eps2pgf_color}\pgfsetfillcolor{eps2pgf_color}
\pgfpathmoveto{\pgfqpoint{0.273cm}{1.395cm}}
\pgfpathcurveto{\pgfqpoint{0.273cm}{1.432cm}}{\pgfqpoint{0.259cm}{1.467cm}}{\pgfqpoint{0.233cm}{1.492cm}}
\pgfpathcurveto{\pgfqpoint{0.207cm}{1.518cm}}{\pgfqpoint{0.173cm}{1.532cm}}{\pgfqpoint{0.137cm}{1.532cm}}
\pgfpathcurveto{\pgfqpoint{0.1cm}{1.532cm}}{\pgfqpoint{0.066cm}{1.518cm}}{\pgfqpoint{0.04cm}{1.492cm}}
\pgfpathcurveto{\pgfqpoint{0.014cm}{1.467cm}}{\pgfqpoint{0cm}{1.432cm}}{\pgfqpoint{0cm}{1.395cm}}
\pgfpathcurveto{\pgfqpoint{0cm}{1.359cm}}{\pgfqpoint{0.014cm}{1.324cm}}{\pgfqpoint{0.04cm}{1.299cm}}
\pgfpathcurveto{\pgfqpoint{0.066cm}{1.273cm}}{\pgfqpoint{0.1cm}{1.258cm}}{\pgfqpoint{0.137cm}{1.258cm}}
\pgfpathcurveto{\pgfqpoint{0.173cm}{1.258cm}}{\pgfqpoint{0.207cm}{1.273cm}}{\pgfqpoint{0.233cm}{1.299cm}}
\pgfpathcurveto{\pgfqpoint{0.259cm}{1.324cm}}{\pgfqpoint{0.273cm}{1.359cm}}{\pgfqpoint{0.273cm}{1.395cm}}
\pgfusepath{fill}
\begin{pgfscope}
\pgfsetdash{}{0cm}
\pgfsetlinewidth{0.818mm}
\pgfsetmiterlimit{7.0}
\pgfpathmoveto{\pgfqpoint{0.682cm}{0.671cm}}
\pgfpathlineto{\pgfqpoint{0.679cm}{1.418cm}}
\pgfusepath{stroke}
\end{pgfscope}
\pgfpathmoveto{\pgfqpoint{0.815cm}{1.399cm}}
\pgfpathcurveto{\pgfqpoint{0.815cm}{1.435cm}}{\pgfqpoint{0.801cm}{1.47cm}}{\pgfqpoint{0.775cm}{1.496cm}}
\pgfpathcurveto{\pgfqpoint{0.75cm}{1.521cm}}{\pgfqpoint{0.715cm}{1.536cm}}{\pgfqpoint{0.679cm}{1.536cm}}
\pgfpathcurveto{\pgfqpoint{0.643cm}{1.536cm}}{\pgfqpoint{0.608cm}{1.521cm}}{\pgfqpoint{0.582cm}{1.496cm}}
\pgfpathcurveto{\pgfqpoint{0.557cm}{1.47cm}}{\pgfqpoint{0.542cm}{1.435cm}}{\pgfqpoint{0.542cm}{1.399cm}}
\pgfpathcurveto{\pgfqpoint{0.542cm}{1.363cm}}{\pgfqpoint{0.557cm}{1.328cm}}{\pgfqpoint{0.582cm}{1.302cm}}
\pgfpathcurveto{\pgfqpoint{0.608cm}{1.276cm}}{\pgfqpoint{0.643cm}{1.262cm}}{\pgfqpoint{0.679cm}{1.262cm}}
\pgfpathcurveto{\pgfqpoint{0.715cm}{1.262cm}}{\pgfqpoint{0.75cm}{1.276cm}}{\pgfqpoint{0.775cm}{1.302cm}}
\pgfpathcurveto{\pgfqpoint{0.801cm}{1.328cm}}{\pgfqpoint{0.815cm}{1.363cm}}{\pgfqpoint{0.815cm}{1.399cm}}
\pgfusepath{fill}
\pgfpathmoveto{\pgfqpoint{1.345cm}{1.371cm}}
\pgfpathcurveto{\pgfqpoint{1.345cm}{1.408cm}}{\pgfqpoint{1.331cm}{1.442cm}}{\pgfqpoint{1.305cm}{1.468cm}}
\pgfpathcurveto{\pgfqpoint{1.28cm}{1.494cm}}{\pgfqpoint{1.245cm}{1.508cm}}{\pgfqpoint{1.209cm}{1.508cm}}
\pgfpathcurveto{\pgfqpoint{1.172cm}{1.508cm}}{\pgfqpoint{1.138cm}{1.494cm}}{\pgfqpoint{1.112cm}{1.468cm}}
\pgfpathcurveto{\pgfqpoint{1.087cm}{1.442cm}}{\pgfqpoint{1.072cm}{1.408cm}}{\pgfqpoint{1.072cm}{1.371cm}}
\pgfpathcurveto{\pgfqpoint{1.072cm}{1.335cm}}{\pgfqpoint{1.087cm}{1.3cm}}{\pgfqpoint{1.112cm}{1.274cm}}
\pgfpathcurveto{\pgfqpoint{1.138cm}{1.249cm}}{\pgfqpoint{1.172cm}{1.234cm}}{\pgfqpoint{1.209cm}{1.234cm}}
\pgfpathcurveto{\pgfqpoint{1.245cm}{1.234cm}}{\pgfqpoint{1.28cm}{1.249cm}}{\pgfqpoint{1.305cm}{1.274cm}}
\pgfpathcurveto{\pgfqpoint{1.331cm}{1.3cm}}{\pgfqpoint{1.345cm}{1.335cm}}{\pgfqpoint{1.345cm}{1.371cm}}
\pgfusepath{fill}
\begin{pgfscope}
\pgfsetdash{}{0cm}
\pgfsetlinewidth{0.818mm}
\pgfsetroundcap
\pgfsetmiterlimit{4.0}
\pgfpathmoveto{\pgfqpoint{0.682cm}{0.671cm}}
\pgfpathlineto{\pgfqpoint{0.682cm}{0.042cm}}
\pgfusepath{stroke}
\end{pgfscope}
\end{pgfscope}
\end{pgfscope}
\end{pgfscope}
\end{tikzpicture}}}_{\varepsilon} = \llbracket X_{\varepsilon}^3 \rrbracket, \quad \Q
   X^{\!\resizebox{0.6em}{!}{
\begin{tikzpicture}
\pgfpathmoveto{\pgfqpoint{0cm}{0cm}}
\pgfpathlineto{\pgfqpoint{1.376cm}{0cm}}
\pgfpathlineto{\pgfqpoint{1.376cm}{1.588cm}}
\pgfpathlineto{\pgfqpoint{0cm}{1.588cm}}
\pgfpathclose
\pgfusepath{clip}
\begin{pgfscope}
\begin{pgfscope}
\pgfpathmoveto{\pgfqpoint{0cm}{0cm}}
\pgfpathlineto{\pgfqpoint{1.376cm}{0cm}}
\pgfpathlineto{\pgfqpoint{1.376cm}{1.588cm}}
\pgfpathlineto{\pgfqpoint{0cm}{1.588cm}}
\pgfpathclose
\pgfusepath{clip}
\begin{pgfscope}
\begin{pgfscope}
\definecolor{eps2pgf_color}{gray}{0.976471}\pgfsetstrokecolor{eps2pgf_color}\pgfsetfillcolor{eps2pgf_color}
\pgfpathmoveto{\pgfqpoint{0cm}{0cm}}
\pgfpathlineto{\pgfqpoint{1.376cm}{0cm}}
\pgfpathlineto{\pgfqpoint{1.376cm}{1.588cm}}
\pgfpathlineto{\pgfqpoint{0cm}{1.588cm}}
\pgfpathclose
\pgfusepath{fill}
\end{pgfscope}
\begin{pgfscope}
\pgfsetdash{}{0cm}
\pgfsetlinewidth{0.818mm}
\pgfsetroundcap
\pgfsetroundjoin
\pgfsetmiterlimit{7.0}
\definecolor{eps2pgf_color}{gray}{0}\pgfsetstrokecolor{eps2pgf_color}\pgfsetfillcolor{eps2pgf_color}
\pgfpathmoveto{\pgfqpoint{0.117cm}{1.476cm}}
\pgfpathlineto{\pgfqpoint{0.682cm}{0.726cm}}
\pgfpathlineto{\pgfqpoint{1.246cm}{1.476cm}}
\pgfusepath{stroke}
\end{pgfscope}
\definecolor{eps2pgf_color}{gray}{0}\pgfsetstrokecolor{eps2pgf_color}\pgfsetfillcolor{eps2pgf_color}
\pgfpathmoveto{\pgfqpoint{0.273cm}{1.451cm}}
\pgfpathcurveto{\pgfqpoint{0.273cm}{1.487cm}}{\pgfqpoint{0.259cm}{1.522cm}}{\pgfqpoint{0.233cm}{1.547cm}}
\pgfpathcurveto{\pgfqpoint{0.207cm}{1.573cm}}{\pgfqpoint{0.173cm}{1.588cm}}{\pgfqpoint{0.137cm}{1.588cm}}
\pgfpathcurveto{\pgfqpoint{0.1cm}{1.588cm}}{\pgfqpoint{0.066cm}{1.573cm}}{\pgfqpoint{0.04cm}{1.547cm}}
\pgfpathcurveto{\pgfqpoint{0.014cm}{1.522cm}}{\pgfqpoint{0cm}{1.487cm}}{\pgfqpoint{0cm}{1.451cm}}
\pgfpathcurveto{\pgfqpoint{0cm}{1.414cm}}{\pgfqpoint{0.014cm}{1.379cm}}{\pgfqpoint{0.04cm}{1.354cm}}
\pgfpathcurveto{\pgfqpoint{0.066cm}{1.328cm}}{\pgfqpoint{0.1cm}{1.314cm}}{\pgfqpoint{0.137cm}{1.314cm}}
\pgfpathcurveto{\pgfqpoint{0.173cm}{1.314cm}}{\pgfqpoint{0.207cm}{1.328cm}}{\pgfqpoint{0.233cm}{1.354cm}}
\pgfpathcurveto{\pgfqpoint{0.259cm}{1.379cm}}{\pgfqpoint{0.273cm}{1.414cm}}{\pgfqpoint{0.273cm}{1.451cm}}
\pgfusepath{fill}
\pgfpathmoveto{\pgfqpoint{1.345cm}{1.426cm}}
\pgfpathcurveto{\pgfqpoint{1.345cm}{1.463cm}}{\pgfqpoint{1.331cm}{1.497cm}}{\pgfqpoint{1.305cm}{1.523cm}}
\pgfpathcurveto{\pgfqpoint{1.28cm}{1.549cm}}{\pgfqpoint{1.245cm}{1.563cm}}{\pgfqpoint{1.209cm}{1.563cm}}
\pgfpathcurveto{\pgfqpoint{1.172cm}{1.563cm}}{\pgfqpoint{1.138cm}{1.549cm}}{\pgfqpoint{1.112cm}{1.523cm}}
\pgfpathcurveto{\pgfqpoint{1.087cm}{1.497cm}}{\pgfqpoint{1.072cm}{1.463cm}}{\pgfqpoint{1.072cm}{1.426cm}}
\pgfpathcurveto{\pgfqpoint{1.072cm}{1.39cm}}{\pgfqpoint{1.087cm}{1.355cm}}{\pgfqpoint{1.112cm}{1.329cm}}
\pgfpathcurveto{\pgfqpoint{1.138cm}{1.304cm}}{\pgfqpoint{1.172cm}{1.289cm}}{\pgfqpoint{1.209cm}{1.289cm}}
\pgfpathcurveto{\pgfqpoint{1.245cm}{1.289cm}}{\pgfqpoint{1.28cm}{1.304cm}}{\pgfqpoint{1.305cm}{1.329cm}}
\pgfpathcurveto{\pgfqpoint{1.331cm}{1.355cm}}{\pgfqpoint{1.345cm}{1.39cm}}{\pgfqpoint{1.345cm}{1.426cm}}
\pgfusepath{fill}
\begin{pgfscope}
\pgfsetdash{}{0cm}
\pgfsetlinewidth{0.818mm}
\pgfsetroundcap
\pgfsetmiterlimit{4.0}
\pgfpathmoveto{\pgfqpoint{0.682cm}{0.726cm}}
\pgfpathlineto{\pgfqpoint{0.682cm}{0.097cm}}
\pgfusepath{stroke}
\end{pgfscope}
\end{pgfscope}
\end{pgfscope}
\end{pgfscope}
\end{tikzpicture}}}_{\varepsilon} = \llbracket X_{\varepsilon}^2 \rrbracket, \]
\[X^{\!\resizebox{!}{.8em}{
\begin{tikzpicture}
\pgfpathmoveto{\pgfqpoint{0cm}{-0.035cm}}
\pgfpathlineto{\pgfqpoint{1.976cm}{-0.035cm}}
\pgfpathlineto{\pgfqpoint{1.976cm}{1.94cm}}
\pgfpathlineto{\pgfqpoint{0cm}{1.94cm}}
\pgfpathclose
\pgfusepath{clip}
\begin{pgfscope}
\begin{pgfscope}
\pgfpathmoveto{\pgfqpoint{0cm}{-0.035cm}}
\pgfpathlineto{\pgfqpoint{1.976cm}{-0.035cm}}
\pgfpathlineto{\pgfqpoint{1.976cm}{1.94cm}}
\pgfpathlineto{\pgfqpoint{0cm}{1.94cm}}
\pgfpathclose
\pgfusepath{clip}
\begin{pgfscope}
\begin{pgfscope}
\pgfsetdash{}{0cm}
\pgfsetlinewidth{0.818mm}
\pgfsetroundcap
\pgfsetroundjoin
\pgfsetmiterlimit{7.0}
\definecolor{eps2pgf_color}{gray}{0}\pgfsetstrokecolor{eps2pgf_color}\pgfsetfillcolor{eps2pgf_color}
\pgfpathmoveto{\pgfqpoint{0.117cm}{1.815cm}}
\pgfpathlineto{\pgfqpoint{0.682cm}{1.065cm}}
\pgfpathlineto{\pgfqpoint{1.246cm}{1.815cm}}
\pgfusepath{stroke}
\end{pgfscope}
\definecolor{eps2pgf_color}{gray}{0}\pgfsetstrokecolor{eps2pgf_color}\pgfsetfillcolor{eps2pgf_color}
\pgfpathmoveto{\pgfqpoint{0.273cm}{1.789cm}}
\pgfpathcurveto{\pgfqpoint{0.273cm}{1.825cm}}{\pgfqpoint{0.259cm}{1.86cm}}{\pgfqpoint{0.233cm}{1.886cm}}
\pgfpathcurveto{\pgfqpoint{0.207cm}{1.912cm}}{\pgfqpoint{0.173cm}{1.926cm}}{\pgfqpoint{0.137cm}{1.926cm}}
\pgfpathcurveto{\pgfqpoint{0.1cm}{1.926cm}}{\pgfqpoint{0.066cm}{1.912cm}}{\pgfqpoint{0.04cm}{1.886cm}}
\pgfpathcurveto{\pgfqpoint{0.014cm}{1.86cm}}{\pgfqpoint{0cm}{1.825cm}}{\pgfqpoint{0cm}{1.789cm}}
\pgfpathcurveto{\pgfqpoint{0cm}{1.753cm}}{\pgfqpoint{0.014cm}{1.718cm}}{\pgfqpoint{0.04cm}{1.692cm}}
\pgfpathcurveto{\pgfqpoint{0.066cm}{1.667cm}}{\pgfqpoint{0.1cm}{1.652cm}}{\pgfqpoint{0.137cm}{1.652cm}}
\pgfpathcurveto{\pgfqpoint{0.173cm}{1.652cm}}{\pgfqpoint{0.207cm}{1.667cm}}{\pgfqpoint{0.233cm}{1.692cm}}
\pgfpathcurveto{\pgfqpoint{0.259cm}{1.718cm}}{\pgfqpoint{0.273cm}{1.753cm}}{\pgfqpoint{0.273cm}{1.789cm}}
\pgfusepath{fill}
\begin{pgfscope}
\pgfsetdash{}{0cm}
\pgfsetlinewidth{0.818mm}
\pgfsetmiterlimit{7.0}
\pgfpathmoveto{\pgfqpoint{0.682cm}{1.065cm}}
\pgfpathlineto{\pgfqpoint{0.679cm}{1.812cm}}
\pgfusepath{stroke}
\end{pgfscope}
\pgfpathmoveto{\pgfqpoint{0.815cm}{1.793cm}}
\pgfpathcurveto{\pgfqpoint{0.815cm}{1.829cm}}{\pgfqpoint{0.801cm}{1.864cm}}{\pgfqpoint{0.775cm}{1.89cm}}
\pgfpathcurveto{\pgfqpoint{0.75cm}{1.915cm}}{\pgfqpoint{0.715cm}{1.93cm}}{\pgfqpoint{0.679cm}{1.93cm}}
\pgfpathcurveto{\pgfqpoint{0.643cm}{1.93cm}}{\pgfqpoint{0.608cm}{1.915cm}}{\pgfqpoint{0.582cm}{1.89cm}}
\pgfpathcurveto{\pgfqpoint{0.557cm}{1.864cm}}{\pgfqpoint{0.542cm}{1.829cm}}{\pgfqpoint{0.542cm}{1.793cm}}
\pgfpathcurveto{\pgfqpoint{0.542cm}{1.756cm}}{\pgfqpoint{0.557cm}{1.722cm}}{\pgfqpoint{0.582cm}{1.696cm}}
\pgfpathcurveto{\pgfqpoint{0.608cm}{1.67cm}}{\pgfqpoint{0.643cm}{1.656cm}}{\pgfqpoint{0.679cm}{1.656cm}}
\pgfpathcurveto{\pgfqpoint{0.715cm}{1.656cm}}{\pgfqpoint{0.75cm}{1.67cm}}{\pgfqpoint{0.775cm}{1.696cm}}
\pgfpathcurveto{\pgfqpoint{0.801cm}{1.722cm}}{\pgfqpoint{0.815cm}{1.756cm}}{\pgfqpoint{0.815cm}{1.793cm}}
\pgfusepath{fill}
\pgfpathmoveto{\pgfqpoint{1.345cm}{1.765cm}}
\pgfpathcurveto{\pgfqpoint{1.345cm}{1.801cm}}{\pgfqpoint{1.331cm}{1.836cm}}{\pgfqpoint{1.305cm}{1.862cm}}
\pgfpathcurveto{\pgfqpoint{1.28cm}{1.887cm}}{\pgfqpoint{1.245cm}{1.902cm}}{\pgfqpoint{1.209cm}{1.902cm}}
\pgfpathcurveto{\pgfqpoint{1.172cm}{1.902cm}}{\pgfqpoint{1.138cm}{1.887cm}}{\pgfqpoint{1.112cm}{1.862cm}}
\pgfpathcurveto{\pgfqpoint{1.087cm}{1.836cm}}{\pgfqpoint{1.072cm}{1.801cm}}{\pgfqpoint{1.072cm}{1.765cm}}
\pgfpathcurveto{\pgfqpoint{1.072cm}{1.728cm}}{\pgfqpoint{1.087cm}{1.694cm}}{\pgfqpoint{1.112cm}{1.668cm}}
\pgfpathcurveto{\pgfqpoint{1.138cm}{1.642cm}}{\pgfqpoint{1.172cm}{1.628cm}}{\pgfqpoint{1.209cm}{1.628cm}}
\pgfpathcurveto{\pgfqpoint{1.245cm}{1.628cm}}{\pgfqpoint{1.28cm}{1.642cm}}{\pgfqpoint{1.305cm}{1.668cm}}
\pgfpathcurveto{\pgfqpoint{1.331cm}{1.694cm}}{\pgfqpoint{1.345cm}{1.728cm}}{\pgfqpoint{1.345cm}{1.765cm}}
\pgfusepath{fill}
\begin{pgfscope}
\pgfsetdash{}{0cm}
\pgfsetlinewidth{0.818mm}
\pgfsetroundcap
\pgfsetroundjoin
\pgfsetmiterlimit{7.0}
\pgfpathmoveto{\pgfqpoint{0.682cm}{1.065cm}}
\pgfpathlineto{\pgfqpoint{1.246cm}{0.315cm}}
\pgfpathlineto{\pgfqpoint{1.811cm}{1.065cm}}
\pgfusepath{stroke}
\end{pgfscope}
\pgfpathmoveto{\pgfqpoint{1.948cm}{1.065cm}}
\pgfpathcurveto{\pgfqpoint{1.948cm}{1.101cm}}{\pgfqpoint{1.933cm}{1.136cm}}{\pgfqpoint{1.907cm}{1.162cm}}
\pgfpathcurveto{\pgfqpoint{1.882cm}{1.187cm}}{\pgfqpoint{1.847cm}{1.202cm}}{\pgfqpoint{1.811cm}{1.202cm}}
\pgfpathcurveto{\pgfqpoint{1.775cm}{1.202cm}}{\pgfqpoint{1.74cm}{1.187cm}}{\pgfqpoint{1.714cm}{1.162cm}}
\pgfpathcurveto{\pgfqpoint{1.689cm}{1.136cm}}{\pgfqpoint{1.674cm}{1.101cm}}{\pgfqpoint{1.674cm}{1.065cm}}
\pgfpathcurveto{\pgfqpoint{1.674cm}{1.029cm}}{\pgfqpoint{1.689cm}{0.994cm}}{\pgfqpoint{1.714cm}{0.968cm}}
\pgfpathcurveto{\pgfqpoint{1.74cm}{0.942cm}}{\pgfqpoint{1.775cm}{0.928cm}}{\pgfqpoint{1.811cm}{0.928cm}}
\pgfpathcurveto{\pgfqpoint{1.847cm}{0.928cm}}{\pgfqpoint{1.882cm}{0.942cm}}{\pgfqpoint{1.907cm}{0.968cm}}
\pgfpathcurveto{\pgfqpoint{1.933cm}{0.994cm}}{\pgfqpoint{1.948cm}{1.029cm}}{\pgfqpoint{1.948cm}{1.065cm}}
\pgfusepath{fill}
\begin{pgfscope}
\pgfsetdash{}{0cm}
\pgfsetlinewidth{0.818mm}
\pgfsetmiterlimit{4.0}
\pgfpathmoveto{\pgfqpoint{1.383cm}{0.178cm}}
\pgfpathcurveto{\pgfqpoint{1.383cm}{0.214cm}}{\pgfqpoint{1.369cm}{0.249cm}}{\pgfqpoint{1.343cm}{0.275cm}}
\pgfpathcurveto{\pgfqpoint{1.317cm}{0.3cm}}{\pgfqpoint{1.283cm}{0.315cm}}{\pgfqpoint{1.246cm}{0.315cm}}
\pgfpathcurveto{\pgfqpoint{1.21cm}{0.315cm}}{\pgfqpoint{1.175cm}{0.3cm}}{\pgfqpoint{1.15cm}{0.275cm}}
\pgfpathcurveto{\pgfqpoint{1.124cm}{0.249cm}}{\pgfqpoint{1.11cm}{0.214cm}}{\pgfqpoint{1.11cm}{0.178cm}}
\pgfpathcurveto{\pgfqpoint{1.11cm}{0.141cm}}{\pgfqpoint{1.124cm}{0.107cm}}{\pgfqpoint{1.15cm}{0.081cm}}
\pgfpathcurveto{\pgfqpoint{1.175cm}{0.055cm}}{\pgfqpoint{1.21cm}{0.041cm}}{\pgfqpoint{1.246cm}{0.041cm}}
\pgfpathcurveto{\pgfqpoint{1.283cm}{0.041cm}}{\pgfqpoint{1.317cm}{0.055cm}}{\pgfqpoint{1.343cm}{0.081cm}}
\pgfpathcurveto{\pgfqpoint{1.369cm}{0.107cm}}{\pgfqpoint{1.383cm}{0.141cm}}{\pgfqpoint{1.383cm}{0.178cm}}
\pgfusepath{stroke}
\end{pgfscope}
\end{pgfscope}
\end{pgfscope}
\end{pgfscope}
\end{tikzpicture}}}=\lim_{\varepsilon\to 0}X^{\!\resizebox{0.6em}{!}{
\begin{tikzpicture}
\pgfpathmoveto{\pgfqpoint{0cm}{-0.035cm}}
\pgfpathlineto{\pgfqpoint{1.376cm}{-0.035cm}}
\pgfpathlineto{\pgfqpoint{1.376cm}{1.552cm}}
\pgfpathlineto{\pgfqpoint{0cm}{1.552cm}}
\pgfpathclose
\pgfusepath{clip}
\begin{pgfscope}
\begin{pgfscope}
\pgfpathmoveto{\pgfqpoint{0cm}{-0.035cm}}
\pgfpathlineto{\pgfqpoint{1.376cm}{-0.035cm}}
\pgfpathlineto{\pgfqpoint{1.376cm}{1.552cm}}
\pgfpathlineto{\pgfqpoint{0cm}{1.552cm}}
\pgfpathclose
\pgfusepath{clip}
\begin{pgfscope}
\begin{pgfscope}
\pgfsetdash{}{0cm}
\pgfsetlinewidth{0.818mm}
\pgfsetroundcap
\pgfsetroundjoin
\pgfsetmiterlimit{7.0}
\definecolor{eps2pgf_color}{gray}{0}\pgfsetstrokecolor{eps2pgf_color}\pgfsetfillcolor{eps2pgf_color}
\pgfpathmoveto{\pgfqpoint{0.117cm}{1.421cm}}
\pgfpathlineto{\pgfqpoint{0.682cm}{0.671cm}}
\pgfpathlineto{\pgfqpoint{1.246cm}{1.421cm}}
\pgfusepath{stroke}
\end{pgfscope}
\definecolor{eps2pgf_color}{gray}{0}\pgfsetstrokecolor{eps2pgf_color}\pgfsetfillcolor{eps2pgf_color}
\pgfpathmoveto{\pgfqpoint{0.273cm}{1.395cm}}
\pgfpathcurveto{\pgfqpoint{0.273cm}{1.432cm}}{\pgfqpoint{0.259cm}{1.467cm}}{\pgfqpoint{0.233cm}{1.492cm}}
\pgfpathcurveto{\pgfqpoint{0.207cm}{1.518cm}}{\pgfqpoint{0.173cm}{1.532cm}}{\pgfqpoint{0.137cm}{1.532cm}}
\pgfpathcurveto{\pgfqpoint{0.1cm}{1.532cm}}{\pgfqpoint{0.066cm}{1.518cm}}{\pgfqpoint{0.04cm}{1.492cm}}
\pgfpathcurveto{\pgfqpoint{0.014cm}{1.467cm}}{\pgfqpoint{0cm}{1.432cm}}{\pgfqpoint{0cm}{1.395cm}}
\pgfpathcurveto{\pgfqpoint{0cm}{1.359cm}}{\pgfqpoint{0.014cm}{1.324cm}}{\pgfqpoint{0.04cm}{1.299cm}}
\pgfpathcurveto{\pgfqpoint{0.066cm}{1.273cm}}{\pgfqpoint{0.1cm}{1.258cm}}{\pgfqpoint{0.137cm}{1.258cm}}
\pgfpathcurveto{\pgfqpoint{0.173cm}{1.258cm}}{\pgfqpoint{0.207cm}{1.273cm}}{\pgfqpoint{0.233cm}{1.299cm}}
\pgfpathcurveto{\pgfqpoint{0.259cm}{1.324cm}}{\pgfqpoint{0.273cm}{1.359cm}}{\pgfqpoint{0.273cm}{1.395cm}}
\pgfusepath{fill}
\begin{pgfscope}
\pgfsetdash{}{0cm}
\pgfsetlinewidth{0.818mm}
\pgfsetmiterlimit{7.0}
\pgfpathmoveto{\pgfqpoint{0.682cm}{0.671cm}}
\pgfpathlineto{\pgfqpoint{0.679cm}{1.418cm}}
\pgfusepath{stroke}
\end{pgfscope}
\pgfpathmoveto{\pgfqpoint{0.815cm}{1.399cm}}
\pgfpathcurveto{\pgfqpoint{0.815cm}{1.435cm}}{\pgfqpoint{0.801cm}{1.47cm}}{\pgfqpoint{0.775cm}{1.496cm}}
\pgfpathcurveto{\pgfqpoint{0.75cm}{1.521cm}}{\pgfqpoint{0.715cm}{1.536cm}}{\pgfqpoint{0.679cm}{1.536cm}}
\pgfpathcurveto{\pgfqpoint{0.643cm}{1.536cm}}{\pgfqpoint{0.608cm}{1.521cm}}{\pgfqpoint{0.582cm}{1.496cm}}
\pgfpathcurveto{\pgfqpoint{0.557cm}{1.47cm}}{\pgfqpoint{0.542cm}{1.435cm}}{\pgfqpoint{0.542cm}{1.399cm}}
\pgfpathcurveto{\pgfqpoint{0.542cm}{1.363cm}}{\pgfqpoint{0.557cm}{1.328cm}}{\pgfqpoint{0.582cm}{1.302cm}}
\pgfpathcurveto{\pgfqpoint{0.608cm}{1.276cm}}{\pgfqpoint{0.643cm}{1.262cm}}{\pgfqpoint{0.679cm}{1.262cm}}
\pgfpathcurveto{\pgfqpoint{0.715cm}{1.262cm}}{\pgfqpoint{0.75cm}{1.276cm}}{\pgfqpoint{0.775cm}{1.302cm}}
\pgfpathcurveto{\pgfqpoint{0.801cm}{1.328cm}}{\pgfqpoint{0.815cm}{1.363cm}}{\pgfqpoint{0.815cm}{1.399cm}}
\pgfusepath{fill}
\pgfpathmoveto{\pgfqpoint{1.345cm}{1.371cm}}
\pgfpathcurveto{\pgfqpoint{1.345cm}{1.408cm}}{\pgfqpoint{1.331cm}{1.442cm}}{\pgfqpoint{1.305cm}{1.468cm}}
\pgfpathcurveto{\pgfqpoint{1.28cm}{1.494cm}}{\pgfqpoint{1.245cm}{1.508cm}}{\pgfqpoint{1.209cm}{1.508cm}}
\pgfpathcurveto{\pgfqpoint{1.172cm}{1.508cm}}{\pgfqpoint{1.138cm}{1.494cm}}{\pgfqpoint{1.112cm}{1.468cm}}
\pgfpathcurveto{\pgfqpoint{1.087cm}{1.442cm}}{\pgfqpoint{1.072cm}{1.408cm}}{\pgfqpoint{1.072cm}{1.371cm}}
\pgfpathcurveto{\pgfqpoint{1.072cm}{1.335cm}}{\pgfqpoint{1.087cm}{1.3cm}}{\pgfqpoint{1.112cm}{1.274cm}}
\pgfpathcurveto{\pgfqpoint{1.138cm}{1.249cm}}{\pgfqpoint{1.172cm}{1.234cm}}{\pgfqpoint{1.209cm}{1.234cm}}
\pgfpathcurveto{\pgfqpoint{1.245cm}{1.234cm}}{\pgfqpoint{1.28cm}{1.249cm}}{\pgfqpoint{1.305cm}{1.274cm}}
\pgfpathcurveto{\pgfqpoint{1.331cm}{1.3cm}}{\pgfqpoint{1.345cm}{1.335cm}}{\pgfqpoint{1.345cm}{1.371cm}}
\pgfusepath{fill}
\begin{pgfscope}
\pgfsetdash{}{0cm}
\pgfsetlinewidth{0.818mm}
\pgfsetroundcap
\pgfsetmiterlimit{4.0}
\pgfpathmoveto{\pgfqpoint{0.682cm}{0.671cm}}
\pgfpathlineto{\pgfqpoint{0.682cm}{0.042cm}}
\pgfusepath{stroke}
\end{pgfscope}
\end{pgfscope}
\end{pgfscope}
\end{pgfscope}
\end{tikzpicture}}}_{\varepsilon}\circ X_{\varepsilon},\qquad X^{\!\resizebox{!}{.8em}{
\begin{tikzpicture}
\pgfpathmoveto{\pgfqpoint{0cm}{-0.035cm}}
\pgfpathlineto{\pgfqpoint{1.976cm}{-0.035cm}}
\pgfpathlineto{\pgfqpoint{1.976cm}{1.94cm}}
\pgfpathlineto{\pgfqpoint{0cm}{1.94cm}}
\pgfpathclose
\pgfusepath{clip}
\begin{pgfscope}
\begin{pgfscope}
\pgfpathmoveto{\pgfqpoint{0cm}{-0.035cm}}
\pgfpathlineto{\pgfqpoint{1.976cm}{-0.035cm}}
\pgfpathlineto{\pgfqpoint{1.976cm}{1.94cm}}
\pgfpathlineto{\pgfqpoint{0cm}{1.94cm}}
\pgfpathclose
\pgfusepath{clip}
\begin{pgfscope}
\begin{pgfscope}
\pgfsetdash{}{0cm}
\pgfsetlinewidth{0.818mm}
\pgfsetroundcap
\pgfsetroundjoin
\pgfsetmiterlimit{7.0}
\definecolor{eps2pgf_color}{gray}{0}\pgfsetstrokecolor{eps2pgf_color}\pgfsetfillcolor{eps2pgf_color}
\pgfpathmoveto{\pgfqpoint{0.117cm}{1.815cm}}
\pgfpathlineto{\pgfqpoint{0.682cm}{1.065cm}}
\pgfpathlineto{\pgfqpoint{1.246cm}{1.815cm}}
\pgfusepath{stroke}
\end{pgfscope}
\definecolor{eps2pgf_color}{gray}{0}\pgfsetstrokecolor{eps2pgf_color}\pgfsetfillcolor{eps2pgf_color}
\pgfpathmoveto{\pgfqpoint{0.273cm}{1.789cm}}
\pgfpathcurveto{\pgfqpoint{0.273cm}{1.825cm}}{\pgfqpoint{0.259cm}{1.86cm}}{\pgfqpoint{0.233cm}{1.886cm}}
\pgfpathcurveto{\pgfqpoint{0.207cm}{1.912cm}}{\pgfqpoint{0.173cm}{1.926cm}}{\pgfqpoint{0.137cm}{1.926cm}}
\pgfpathcurveto{\pgfqpoint{0.1cm}{1.926cm}}{\pgfqpoint{0.066cm}{1.912cm}}{\pgfqpoint{0.04cm}{1.886cm}}
\pgfpathcurveto{\pgfqpoint{0.014cm}{1.86cm}}{\pgfqpoint{0cm}{1.825cm}}{\pgfqpoint{0cm}{1.789cm}}
\pgfpathcurveto{\pgfqpoint{0cm}{1.753cm}}{\pgfqpoint{0.014cm}{1.718cm}}{\pgfqpoint{0.04cm}{1.692cm}}
\pgfpathcurveto{\pgfqpoint{0.066cm}{1.667cm}}{\pgfqpoint{0.1cm}{1.652cm}}{\pgfqpoint{0.137cm}{1.652cm}}
\pgfpathcurveto{\pgfqpoint{0.173cm}{1.652cm}}{\pgfqpoint{0.207cm}{1.667cm}}{\pgfqpoint{0.233cm}{1.692cm}}
\pgfpathcurveto{\pgfqpoint{0.259cm}{1.718cm}}{\pgfqpoint{0.273cm}{1.753cm}}{\pgfqpoint{0.273cm}{1.789cm}}
\pgfusepath{fill}
\pgfpathmoveto{\pgfqpoint{1.345cm}{1.765cm}}
\pgfpathcurveto{\pgfqpoint{1.345cm}{1.801cm}}{\pgfqpoint{1.331cm}{1.836cm}}{\pgfqpoint{1.305cm}{1.862cm}}
\pgfpathcurveto{\pgfqpoint{1.28cm}{1.887cm}}{\pgfqpoint{1.245cm}{1.902cm}}{\pgfqpoint{1.209cm}{1.902cm}}
\pgfpathcurveto{\pgfqpoint{1.172cm}{1.902cm}}{\pgfqpoint{1.138cm}{1.887cm}}{\pgfqpoint{1.112cm}{1.862cm}}
\pgfpathcurveto{\pgfqpoint{1.087cm}{1.836cm}}{\pgfqpoint{1.072cm}{1.801cm}}{\pgfqpoint{1.072cm}{1.765cm}}
\pgfpathcurveto{\pgfqpoint{1.072cm}{1.728cm}}{\pgfqpoint{1.087cm}{1.694cm}}{\pgfqpoint{1.112cm}{1.668cm}}
\pgfpathcurveto{\pgfqpoint{1.138cm}{1.642cm}}{\pgfqpoint{1.172cm}{1.628cm}}{\pgfqpoint{1.209cm}{1.628cm}}
\pgfpathcurveto{\pgfqpoint{1.245cm}{1.628cm}}{\pgfqpoint{1.28cm}{1.642cm}}{\pgfqpoint{1.305cm}{1.668cm}}
\pgfpathcurveto{\pgfqpoint{1.331cm}{1.694cm}}{\pgfqpoint{1.345cm}{1.728cm}}{\pgfqpoint{1.345cm}{1.765cm}}
\pgfusepath{fill}
\begin{pgfscope}
\pgfsetdash{}{0cm}
\pgfsetlinewidth{0.818mm}
\pgfsetroundcap
\pgfsetroundjoin
\pgfsetmiterlimit{7.0}
\pgfpathmoveto{\pgfqpoint{0.682cm}{1.065cm}}
\pgfpathlineto{\pgfqpoint{1.246cm}{0.315cm}}
\pgfpathlineto{\pgfqpoint{1.811cm}{1.065cm}}
\pgfusepath{stroke}
\end{pgfscope}
\pgfpathmoveto{\pgfqpoint{1.948cm}{1.065cm}}
\pgfpathcurveto{\pgfqpoint{1.948cm}{1.101cm}}{\pgfqpoint{1.933cm}{1.136cm}}{\pgfqpoint{1.907cm}{1.162cm}}
\pgfpathcurveto{\pgfqpoint{1.882cm}{1.187cm}}{\pgfqpoint{1.847cm}{1.202cm}}{\pgfqpoint{1.811cm}{1.202cm}}
\pgfpathcurveto{\pgfqpoint{1.775cm}{1.202cm}}{\pgfqpoint{1.74cm}{1.187cm}}{\pgfqpoint{1.714cm}{1.162cm}}
\pgfpathcurveto{\pgfqpoint{1.689cm}{1.136cm}}{\pgfqpoint{1.674cm}{1.101cm}}{\pgfqpoint{1.674cm}{1.065cm}}
\pgfpathcurveto{\pgfqpoint{1.674cm}{1.029cm}}{\pgfqpoint{1.689cm}{0.994cm}}{\pgfqpoint{1.714cm}{0.968cm}}
\pgfpathcurveto{\pgfqpoint{1.74cm}{0.942cm}}{\pgfqpoint{1.775cm}{0.928cm}}{\pgfqpoint{1.811cm}{0.928cm}}
\pgfpathcurveto{\pgfqpoint{1.847cm}{0.928cm}}{\pgfqpoint{1.882cm}{0.942cm}}{\pgfqpoint{1.907cm}{0.968cm}}
\pgfpathcurveto{\pgfqpoint{1.933cm}{0.994cm}}{\pgfqpoint{1.948cm}{1.029cm}}{\pgfqpoint{1.948cm}{1.065cm}}
\pgfusepath{fill}
\begin{pgfscope}
\pgfsetdash{}{0cm}
\pgfsetlinewidth{0.818mm}
\pgfsetmiterlimit{7.0}
\pgfpathmoveto{\pgfqpoint{1.246cm}{0.315cm}}
\pgfpathlineto{\pgfqpoint{1.244cm}{1.061cm}}
\pgfusepath{stroke}
\end{pgfscope}
\pgfpathmoveto{\pgfqpoint{1.38cm}{1.065cm}}
\pgfpathcurveto{\pgfqpoint{1.38cm}{1.101cm}}{\pgfqpoint{1.366cm}{1.136cm}}{\pgfqpoint{1.34cm}{1.162cm}}
\pgfpathcurveto{\pgfqpoint{1.315cm}{1.187cm}}{\pgfqpoint{1.28cm}{1.202cm}}{\pgfqpoint{1.244cm}{1.202cm}}
\pgfpathcurveto{\pgfqpoint{1.207cm}{1.202cm}}{\pgfqpoint{1.173cm}{1.187cm}}{\pgfqpoint{1.147cm}{1.162cm}}
\pgfpathcurveto{\pgfqpoint{1.121cm}{1.136cm}}{\pgfqpoint{1.107cm}{1.101cm}}{\pgfqpoint{1.107cm}{1.065cm}}
\pgfpathcurveto{\pgfqpoint{1.107cm}{1.029cm}}{\pgfqpoint{1.121cm}{0.994cm}}{\pgfqpoint{1.147cm}{0.968cm}}
\pgfpathcurveto{\pgfqpoint{1.173cm}{0.942cm}}{\pgfqpoint{1.207cm}{0.928cm}}{\pgfqpoint{1.244cm}{0.928cm}}
\pgfpathcurveto{\pgfqpoint{1.28cm}{0.928cm}}{\pgfqpoint{1.315cm}{0.942cm}}{\pgfqpoint{1.34cm}{0.968cm}}
\pgfpathcurveto{\pgfqpoint{1.366cm}{0.994cm}}{\pgfqpoint{1.38cm}{1.029cm}}{\pgfqpoint{1.38cm}{1.065cm}}
\pgfusepath{fill}
\begin{pgfscope}
\pgfsetdash{}{0cm}
\pgfsetlinewidth{0.818mm}
\pgfsetmiterlimit{4.0}
\pgfpathmoveto{\pgfqpoint{1.383cm}{0.178cm}}
\pgfpathcurveto{\pgfqpoint{1.383cm}{0.214cm}}{\pgfqpoint{1.369cm}{0.249cm}}{\pgfqpoint{1.343cm}{0.275cm}}
\pgfpathcurveto{\pgfqpoint{1.317cm}{0.3cm}}{\pgfqpoint{1.283cm}{0.315cm}}{\pgfqpoint{1.246cm}{0.315cm}}
\pgfpathcurveto{\pgfqpoint{1.21cm}{0.315cm}}{\pgfqpoint{1.175cm}{0.3cm}}{\pgfqpoint{1.15cm}{0.275cm}}
\pgfpathcurveto{\pgfqpoint{1.124cm}{0.249cm}}{\pgfqpoint{1.11cm}{0.214cm}}{\pgfqpoint{1.11cm}{0.178cm}}
\pgfpathcurveto{\pgfqpoint{1.11cm}{0.141cm}}{\pgfqpoint{1.124cm}{0.107cm}}{\pgfqpoint{1.15cm}{0.081cm}}
\pgfpathcurveto{\pgfqpoint{1.175cm}{0.055cm}}{\pgfqpoint{1.21cm}{0.041cm}}{\pgfqpoint{1.246cm}{0.041cm}}
\pgfpathcurveto{\pgfqpoint{1.283cm}{0.041cm}}{\pgfqpoint{1.317cm}{0.055cm}}{\pgfqpoint{1.343cm}{0.081cm}}
\pgfpathcurveto{\pgfqpoint{1.369cm}{0.107cm}}{\pgfqpoint{1.383cm}{0.141cm}}{\pgfqpoint{1.383cm}{0.178cm}}
\pgfusepath{stroke}
\end{pgfscope}
\end{pgfscope}
\end{pgfscope}
\end{pgfscope}
\end{tikzpicture}}} = \lim_{\varepsilon\to 0}X^{\!\resizebox{0.6em}{!}{
\begin{tikzpicture}
\pgfpathmoveto{\pgfqpoint{0cm}{0cm}}
\pgfpathlineto{\pgfqpoint{1.376cm}{0cm}}
\pgfpathlineto{\pgfqpoint{1.376cm}{1.588cm}}
\pgfpathlineto{\pgfqpoint{0cm}{1.588cm}}
\pgfpathclose
\pgfusepath{clip}
\begin{pgfscope}
\begin{pgfscope}
\pgfpathmoveto{\pgfqpoint{0cm}{0cm}}
\pgfpathlineto{\pgfqpoint{1.376cm}{0cm}}
\pgfpathlineto{\pgfqpoint{1.376cm}{1.588cm}}
\pgfpathlineto{\pgfqpoint{0cm}{1.588cm}}
\pgfpathclose
\pgfusepath{clip}
\begin{pgfscope}
\begin{pgfscope}
\definecolor{eps2pgf_color}{gray}{0.976471}\pgfsetstrokecolor{eps2pgf_color}\pgfsetfillcolor{eps2pgf_color}
\pgfpathmoveto{\pgfqpoint{0cm}{0cm}}
\pgfpathlineto{\pgfqpoint{1.376cm}{0cm}}
\pgfpathlineto{\pgfqpoint{1.376cm}{1.588cm}}
\pgfpathlineto{\pgfqpoint{0cm}{1.588cm}}
\pgfpathclose
\pgfusepath{fill}
\end{pgfscope}
\begin{pgfscope}
\pgfsetdash{}{0cm}
\pgfsetlinewidth{0.818mm}
\pgfsetroundcap
\pgfsetroundjoin
\pgfsetmiterlimit{7.0}
\definecolor{eps2pgf_color}{gray}{0}\pgfsetstrokecolor{eps2pgf_color}\pgfsetfillcolor{eps2pgf_color}
\pgfpathmoveto{\pgfqpoint{0.117cm}{1.476cm}}
\pgfpathlineto{\pgfqpoint{0.682cm}{0.726cm}}
\pgfpathlineto{\pgfqpoint{1.246cm}{1.476cm}}
\pgfusepath{stroke}
\end{pgfscope}
\definecolor{eps2pgf_color}{gray}{0}\pgfsetstrokecolor{eps2pgf_color}\pgfsetfillcolor{eps2pgf_color}
\pgfpathmoveto{\pgfqpoint{0.273cm}{1.451cm}}
\pgfpathcurveto{\pgfqpoint{0.273cm}{1.487cm}}{\pgfqpoint{0.259cm}{1.522cm}}{\pgfqpoint{0.233cm}{1.547cm}}
\pgfpathcurveto{\pgfqpoint{0.207cm}{1.573cm}}{\pgfqpoint{0.173cm}{1.588cm}}{\pgfqpoint{0.137cm}{1.588cm}}
\pgfpathcurveto{\pgfqpoint{0.1cm}{1.588cm}}{\pgfqpoint{0.066cm}{1.573cm}}{\pgfqpoint{0.04cm}{1.547cm}}
\pgfpathcurveto{\pgfqpoint{0.014cm}{1.522cm}}{\pgfqpoint{0cm}{1.487cm}}{\pgfqpoint{0cm}{1.451cm}}
\pgfpathcurveto{\pgfqpoint{0cm}{1.414cm}}{\pgfqpoint{0.014cm}{1.379cm}}{\pgfqpoint{0.04cm}{1.354cm}}
\pgfpathcurveto{\pgfqpoint{0.066cm}{1.328cm}}{\pgfqpoint{0.1cm}{1.314cm}}{\pgfqpoint{0.137cm}{1.314cm}}
\pgfpathcurveto{\pgfqpoint{0.173cm}{1.314cm}}{\pgfqpoint{0.207cm}{1.328cm}}{\pgfqpoint{0.233cm}{1.354cm}}
\pgfpathcurveto{\pgfqpoint{0.259cm}{1.379cm}}{\pgfqpoint{0.273cm}{1.414cm}}{\pgfqpoint{0.273cm}{1.451cm}}
\pgfusepath{fill}
\pgfpathmoveto{\pgfqpoint{1.345cm}{1.426cm}}
\pgfpathcurveto{\pgfqpoint{1.345cm}{1.463cm}}{\pgfqpoint{1.331cm}{1.497cm}}{\pgfqpoint{1.305cm}{1.523cm}}
\pgfpathcurveto{\pgfqpoint{1.28cm}{1.549cm}}{\pgfqpoint{1.245cm}{1.563cm}}{\pgfqpoint{1.209cm}{1.563cm}}
\pgfpathcurveto{\pgfqpoint{1.172cm}{1.563cm}}{\pgfqpoint{1.138cm}{1.549cm}}{\pgfqpoint{1.112cm}{1.523cm}}
\pgfpathcurveto{\pgfqpoint{1.087cm}{1.497cm}}{\pgfqpoint{1.072cm}{1.463cm}}{\pgfqpoint{1.072cm}{1.426cm}}
\pgfpathcurveto{\pgfqpoint{1.072cm}{1.39cm}}{\pgfqpoint{1.087cm}{1.355cm}}{\pgfqpoint{1.112cm}{1.329cm}}
\pgfpathcurveto{\pgfqpoint{1.138cm}{1.304cm}}{\pgfqpoint{1.172cm}{1.289cm}}{\pgfqpoint{1.209cm}{1.289cm}}
\pgfpathcurveto{\pgfqpoint{1.245cm}{1.289cm}}{\pgfqpoint{1.28cm}{1.304cm}}{\pgfqpoint{1.305cm}{1.329cm}}
\pgfpathcurveto{\pgfqpoint{1.331cm}{1.355cm}}{\pgfqpoint{1.345cm}{1.39cm}}{\pgfqpoint{1.345cm}{1.426cm}}
\pgfusepath{fill}
\begin{pgfscope}
\pgfsetdash{}{0cm}
\pgfsetlinewidth{0.818mm}
\pgfsetroundcap
\pgfsetmiterlimit{4.0}
\pgfpathmoveto{\pgfqpoint{0.682cm}{0.726cm}}
\pgfpathlineto{\pgfqpoint{0.682cm}{0.097cm}}
\pgfusepath{stroke}
\end{pgfscope}
\end{pgfscope}
\end{pgfscope}
\end{pgfscope}
\end{tikzpicture}}}_{\varepsilon} \circ \llbracket X_{\varepsilon}^2 \rrbracket -\frac{b_{\varepsilon}}{3}, \qquad
   X^{\!\resizebox{!}{.8em}{
\begin{tikzpicture}
\pgfpathmoveto{\pgfqpoint{0cm}{-0.035cm}}
\pgfpathlineto{\pgfqpoint{1.976cm}{-0.035cm}}
\pgfpathlineto{\pgfqpoint{1.976cm}{1.94cm}}
\pgfpathlineto{\pgfqpoint{0cm}{1.94cm}}
\pgfpathclose
\pgfusepath{clip}
\begin{pgfscope}
\begin{pgfscope}
\pgfpathmoveto{\pgfqpoint{0cm}{-0.035cm}}
\pgfpathlineto{\pgfqpoint{1.976cm}{-0.035cm}}
\pgfpathlineto{\pgfqpoint{1.976cm}{1.94cm}}
\pgfpathlineto{\pgfqpoint{0cm}{1.94cm}}
\pgfpathclose
\pgfusepath{clip}
\begin{pgfscope}
\begin{pgfscope}
\pgfsetdash{}{0cm}
\pgfsetlinewidth{0.818mm}
\pgfsetroundcap
\pgfsetroundjoin
\pgfsetmiterlimit{7.0}
\definecolor{eps2pgf_color}{gray}{0}\pgfsetstrokecolor{eps2pgf_color}\pgfsetfillcolor{eps2pgf_color}
\pgfpathmoveto{\pgfqpoint{0.117cm}{1.815cm}}
\pgfpathlineto{\pgfqpoint{0.682cm}{1.065cm}}
\pgfpathlineto{\pgfqpoint{1.246cm}{1.815cm}}
\pgfusepath{stroke}
\end{pgfscope}
\definecolor{eps2pgf_color}{gray}{0}\pgfsetstrokecolor{eps2pgf_color}\pgfsetfillcolor{eps2pgf_color}
\pgfpathmoveto{\pgfqpoint{0.273cm}{1.789cm}}
\pgfpathcurveto{\pgfqpoint{0.273cm}{1.825cm}}{\pgfqpoint{0.259cm}{1.86cm}}{\pgfqpoint{0.233cm}{1.886cm}}
\pgfpathcurveto{\pgfqpoint{0.207cm}{1.912cm}}{\pgfqpoint{0.173cm}{1.926cm}}{\pgfqpoint{0.137cm}{1.926cm}}
\pgfpathcurveto{\pgfqpoint{0.1cm}{1.926cm}}{\pgfqpoint{0.066cm}{1.912cm}}{\pgfqpoint{0.04cm}{1.886cm}}
\pgfpathcurveto{\pgfqpoint{0.014cm}{1.86cm}}{\pgfqpoint{0cm}{1.825cm}}{\pgfqpoint{0cm}{1.789cm}}
\pgfpathcurveto{\pgfqpoint{0cm}{1.753cm}}{\pgfqpoint{0.014cm}{1.718cm}}{\pgfqpoint{0.04cm}{1.692cm}}
\pgfpathcurveto{\pgfqpoint{0.066cm}{1.667cm}}{\pgfqpoint{0.1cm}{1.652cm}}{\pgfqpoint{0.137cm}{1.652cm}}
\pgfpathcurveto{\pgfqpoint{0.173cm}{1.652cm}}{\pgfqpoint{0.207cm}{1.667cm}}{\pgfqpoint{0.233cm}{1.692cm}}
\pgfpathcurveto{\pgfqpoint{0.259cm}{1.718cm}}{\pgfqpoint{0.273cm}{1.753cm}}{\pgfqpoint{0.273cm}{1.789cm}}
\pgfusepath{fill}
\begin{pgfscope}
\pgfsetdash{}{0cm}
\pgfsetlinewidth{0.818mm}
\pgfsetmiterlimit{7.0}
\pgfpathmoveto{\pgfqpoint{0.682cm}{1.065cm}}
\pgfpathlineto{\pgfqpoint{0.679cm}{1.812cm}}
\pgfusepath{stroke}
\end{pgfscope}
\pgfpathmoveto{\pgfqpoint{0.815cm}{1.793cm}}
\pgfpathcurveto{\pgfqpoint{0.815cm}{1.829cm}}{\pgfqpoint{0.801cm}{1.864cm}}{\pgfqpoint{0.775cm}{1.89cm}}
\pgfpathcurveto{\pgfqpoint{0.75cm}{1.915cm}}{\pgfqpoint{0.715cm}{1.93cm}}{\pgfqpoint{0.679cm}{1.93cm}}
\pgfpathcurveto{\pgfqpoint{0.643cm}{1.93cm}}{\pgfqpoint{0.608cm}{1.915cm}}{\pgfqpoint{0.582cm}{1.89cm}}
\pgfpathcurveto{\pgfqpoint{0.557cm}{1.864cm}}{\pgfqpoint{0.542cm}{1.829cm}}{\pgfqpoint{0.542cm}{1.793cm}}
\pgfpathcurveto{\pgfqpoint{0.542cm}{1.756cm}}{\pgfqpoint{0.557cm}{1.722cm}}{\pgfqpoint{0.582cm}{1.696cm}}
\pgfpathcurveto{\pgfqpoint{0.608cm}{1.67cm}}{\pgfqpoint{0.643cm}{1.656cm}}{\pgfqpoint{0.679cm}{1.656cm}}
\pgfpathcurveto{\pgfqpoint{0.715cm}{1.656cm}}{\pgfqpoint{0.75cm}{1.67cm}}{\pgfqpoint{0.775cm}{1.696cm}}
\pgfpathcurveto{\pgfqpoint{0.801cm}{1.722cm}}{\pgfqpoint{0.815cm}{1.756cm}}{\pgfqpoint{0.815cm}{1.793cm}}
\pgfusepath{fill}
\pgfpathmoveto{\pgfqpoint{1.345cm}{1.765cm}}
\pgfpathcurveto{\pgfqpoint{1.345cm}{1.801cm}}{\pgfqpoint{1.331cm}{1.836cm}}{\pgfqpoint{1.305cm}{1.862cm}}
\pgfpathcurveto{\pgfqpoint{1.28cm}{1.887cm}}{\pgfqpoint{1.245cm}{1.902cm}}{\pgfqpoint{1.209cm}{1.902cm}}
\pgfpathcurveto{\pgfqpoint{1.172cm}{1.902cm}}{\pgfqpoint{1.138cm}{1.887cm}}{\pgfqpoint{1.112cm}{1.862cm}}
\pgfpathcurveto{\pgfqpoint{1.087cm}{1.836cm}}{\pgfqpoint{1.072cm}{1.801cm}}{\pgfqpoint{1.072cm}{1.765cm}}
\pgfpathcurveto{\pgfqpoint{1.072cm}{1.728cm}}{\pgfqpoint{1.087cm}{1.694cm}}{\pgfqpoint{1.112cm}{1.668cm}}
\pgfpathcurveto{\pgfqpoint{1.138cm}{1.642cm}}{\pgfqpoint{1.172cm}{1.628cm}}{\pgfqpoint{1.209cm}{1.628cm}}
\pgfpathcurveto{\pgfqpoint{1.245cm}{1.628cm}}{\pgfqpoint{1.28cm}{1.642cm}}{\pgfqpoint{1.305cm}{1.668cm}}
\pgfpathcurveto{\pgfqpoint{1.331cm}{1.694cm}}{\pgfqpoint{1.345cm}{1.728cm}}{\pgfqpoint{1.345cm}{1.765cm}}
\pgfusepath{fill}
\begin{pgfscope}
\pgfsetdash{}{0cm}
\pgfsetlinewidth{0.818mm}
\pgfsetroundcap
\pgfsetroundjoin
\pgfsetmiterlimit{7.0}
\pgfpathmoveto{\pgfqpoint{0.682cm}{1.065cm}}
\pgfpathlineto{\pgfqpoint{1.246cm}{0.315cm}}
\pgfpathlineto{\pgfqpoint{1.811cm}{1.065cm}}
\pgfusepath{stroke}
\end{pgfscope}
\pgfpathmoveto{\pgfqpoint{1.948cm}{1.065cm}}
\pgfpathcurveto{\pgfqpoint{1.948cm}{1.101cm}}{\pgfqpoint{1.933cm}{1.136cm}}{\pgfqpoint{1.907cm}{1.162cm}}
\pgfpathcurveto{\pgfqpoint{1.882cm}{1.187cm}}{\pgfqpoint{1.847cm}{1.202cm}}{\pgfqpoint{1.811cm}{1.202cm}}
\pgfpathcurveto{\pgfqpoint{1.775cm}{1.202cm}}{\pgfqpoint{1.74cm}{1.187cm}}{\pgfqpoint{1.714cm}{1.162cm}}
\pgfpathcurveto{\pgfqpoint{1.689cm}{1.136cm}}{\pgfqpoint{1.674cm}{1.101cm}}{\pgfqpoint{1.674cm}{1.065cm}}
\pgfpathcurveto{\pgfqpoint{1.674cm}{1.029cm}}{\pgfqpoint{1.689cm}{0.994cm}}{\pgfqpoint{1.714cm}{0.968cm}}
\pgfpathcurveto{\pgfqpoint{1.74cm}{0.942cm}}{\pgfqpoint{1.775cm}{0.928cm}}{\pgfqpoint{1.811cm}{0.928cm}}
\pgfpathcurveto{\pgfqpoint{1.847cm}{0.928cm}}{\pgfqpoint{1.882cm}{0.942cm}}{\pgfqpoint{1.907cm}{0.968cm}}
\pgfpathcurveto{\pgfqpoint{1.933cm}{0.994cm}}{\pgfqpoint{1.948cm}{1.029cm}}{\pgfqpoint{1.948cm}{1.065cm}}
\pgfusepath{fill}
\begin{pgfscope}
\pgfsetdash{}{0cm}
\pgfsetlinewidth{0.818mm}
\pgfsetmiterlimit{7.0}
\pgfpathmoveto{\pgfqpoint{1.246cm}{0.315cm}}
\pgfpathlineto{\pgfqpoint{1.244cm}{1.061cm}}
\pgfusepath{stroke}
\end{pgfscope}
\pgfpathmoveto{\pgfqpoint{1.38cm}{1.065cm}}
\pgfpathcurveto{\pgfqpoint{1.38cm}{1.101cm}}{\pgfqpoint{1.366cm}{1.136cm}}{\pgfqpoint{1.34cm}{1.162cm}}
\pgfpathcurveto{\pgfqpoint{1.315cm}{1.187cm}}{\pgfqpoint{1.28cm}{1.202cm}}{\pgfqpoint{1.244cm}{1.202cm}}
\pgfpathcurveto{\pgfqpoint{1.207cm}{1.202cm}}{\pgfqpoint{1.173cm}{1.187cm}}{\pgfqpoint{1.147cm}{1.162cm}}
\pgfpathcurveto{\pgfqpoint{1.121cm}{1.136cm}}{\pgfqpoint{1.107cm}{1.101cm}}{\pgfqpoint{1.107cm}{1.065cm}}
\pgfpathcurveto{\pgfqpoint{1.107cm}{1.029cm}}{\pgfqpoint{1.121cm}{0.994cm}}{\pgfqpoint{1.147cm}{0.968cm}}
\pgfpathcurveto{\pgfqpoint{1.173cm}{0.942cm}}{\pgfqpoint{1.207cm}{0.928cm}}{\pgfqpoint{1.244cm}{0.928cm}}
\pgfpathcurveto{\pgfqpoint{1.28cm}{0.928cm}}{\pgfqpoint{1.315cm}{0.942cm}}{\pgfqpoint{1.34cm}{0.968cm}}
\pgfpathcurveto{\pgfqpoint{1.366cm}{0.994cm}}{\pgfqpoint{1.38cm}{1.029cm}}{\pgfqpoint{1.38cm}{1.065cm}}
\pgfusepath{fill}
\begin{pgfscope}
\pgfsetdash{}{0cm}
\pgfsetlinewidth{0.818mm}
\pgfsetmiterlimit{4.0}
\pgfpathmoveto{\pgfqpoint{1.383cm}{0.178cm}}
\pgfpathcurveto{\pgfqpoint{1.383cm}{0.214cm}}{\pgfqpoint{1.369cm}{0.249cm}}{\pgfqpoint{1.343cm}{0.275cm}}
\pgfpathcurveto{\pgfqpoint{1.317cm}{0.3cm}}{\pgfqpoint{1.283cm}{0.315cm}}{\pgfqpoint{1.246cm}{0.315cm}}
\pgfpathcurveto{\pgfqpoint{1.21cm}{0.315cm}}{\pgfqpoint{1.175cm}{0.3cm}}{\pgfqpoint{1.15cm}{0.275cm}}
\pgfpathcurveto{\pgfqpoint{1.124cm}{0.249cm}}{\pgfqpoint{1.11cm}{0.214cm}}{\pgfqpoint{1.11cm}{0.178cm}}
\pgfpathcurveto{\pgfqpoint{1.11cm}{0.141cm}}{\pgfqpoint{1.124cm}{0.107cm}}{\pgfqpoint{1.15cm}{0.081cm}}
\pgfpathcurveto{\pgfqpoint{1.175cm}{0.055cm}}{\pgfqpoint{1.21cm}{0.041cm}}{\pgfqpoint{1.246cm}{0.041cm}}
\pgfpathcurveto{\pgfqpoint{1.283cm}{0.041cm}}{\pgfqpoint{1.317cm}{0.055cm}}{\pgfqpoint{1.343cm}{0.081cm}}
\pgfpathcurveto{\pgfqpoint{1.369cm}{0.107cm}}{\pgfqpoint{1.383cm}{0.141cm}}{\pgfqpoint{1.383cm}{0.178cm}}
\pgfusepath{stroke}
\end{pgfscope}
\end{pgfscope}
\end{pgfscope}
\end{pgfscope}
\end{tikzpicture}}} = \lim_{\varepsilon\to 0}X^{\!\resizebox{0.6em}{!}{
\begin{tikzpicture}
\pgfpathmoveto{\pgfqpoint{0cm}{-0.035cm}}
\pgfpathlineto{\pgfqpoint{1.376cm}{-0.035cm}}
\pgfpathlineto{\pgfqpoint{1.376cm}{1.552cm}}
\pgfpathlineto{\pgfqpoint{0cm}{1.552cm}}
\pgfpathclose
\pgfusepath{clip}
\begin{pgfscope}
\begin{pgfscope}
\pgfpathmoveto{\pgfqpoint{0cm}{-0.035cm}}
\pgfpathlineto{\pgfqpoint{1.376cm}{-0.035cm}}
\pgfpathlineto{\pgfqpoint{1.376cm}{1.552cm}}
\pgfpathlineto{\pgfqpoint{0cm}{1.552cm}}
\pgfpathclose
\pgfusepath{clip}
\begin{pgfscope}
\begin{pgfscope}
\pgfsetdash{}{0cm}
\pgfsetlinewidth{0.818mm}
\pgfsetroundcap
\pgfsetroundjoin
\pgfsetmiterlimit{7.0}
\definecolor{eps2pgf_color}{gray}{0}\pgfsetstrokecolor{eps2pgf_color}\pgfsetfillcolor{eps2pgf_color}
\pgfpathmoveto{\pgfqpoint{0.117cm}{1.421cm}}
\pgfpathlineto{\pgfqpoint{0.682cm}{0.671cm}}
\pgfpathlineto{\pgfqpoint{1.246cm}{1.421cm}}
\pgfusepath{stroke}
\end{pgfscope}
\definecolor{eps2pgf_color}{gray}{0}\pgfsetstrokecolor{eps2pgf_color}\pgfsetfillcolor{eps2pgf_color}
\pgfpathmoveto{\pgfqpoint{0.273cm}{1.395cm}}
\pgfpathcurveto{\pgfqpoint{0.273cm}{1.432cm}}{\pgfqpoint{0.259cm}{1.467cm}}{\pgfqpoint{0.233cm}{1.492cm}}
\pgfpathcurveto{\pgfqpoint{0.207cm}{1.518cm}}{\pgfqpoint{0.173cm}{1.532cm}}{\pgfqpoint{0.137cm}{1.532cm}}
\pgfpathcurveto{\pgfqpoint{0.1cm}{1.532cm}}{\pgfqpoint{0.066cm}{1.518cm}}{\pgfqpoint{0.04cm}{1.492cm}}
\pgfpathcurveto{\pgfqpoint{0.014cm}{1.467cm}}{\pgfqpoint{0cm}{1.432cm}}{\pgfqpoint{0cm}{1.395cm}}
\pgfpathcurveto{\pgfqpoint{0cm}{1.359cm}}{\pgfqpoint{0.014cm}{1.324cm}}{\pgfqpoint{0.04cm}{1.299cm}}
\pgfpathcurveto{\pgfqpoint{0.066cm}{1.273cm}}{\pgfqpoint{0.1cm}{1.258cm}}{\pgfqpoint{0.137cm}{1.258cm}}
\pgfpathcurveto{\pgfqpoint{0.173cm}{1.258cm}}{\pgfqpoint{0.207cm}{1.273cm}}{\pgfqpoint{0.233cm}{1.299cm}}
\pgfpathcurveto{\pgfqpoint{0.259cm}{1.324cm}}{\pgfqpoint{0.273cm}{1.359cm}}{\pgfqpoint{0.273cm}{1.395cm}}
\pgfusepath{fill}
\begin{pgfscope}
\pgfsetdash{}{0cm}
\pgfsetlinewidth{0.818mm}
\pgfsetmiterlimit{7.0}
\pgfpathmoveto{\pgfqpoint{0.682cm}{0.671cm}}
\pgfpathlineto{\pgfqpoint{0.679cm}{1.418cm}}
\pgfusepath{stroke}
\end{pgfscope}
\pgfpathmoveto{\pgfqpoint{0.815cm}{1.399cm}}
\pgfpathcurveto{\pgfqpoint{0.815cm}{1.435cm}}{\pgfqpoint{0.801cm}{1.47cm}}{\pgfqpoint{0.775cm}{1.496cm}}
\pgfpathcurveto{\pgfqpoint{0.75cm}{1.521cm}}{\pgfqpoint{0.715cm}{1.536cm}}{\pgfqpoint{0.679cm}{1.536cm}}
\pgfpathcurveto{\pgfqpoint{0.643cm}{1.536cm}}{\pgfqpoint{0.608cm}{1.521cm}}{\pgfqpoint{0.582cm}{1.496cm}}
\pgfpathcurveto{\pgfqpoint{0.557cm}{1.47cm}}{\pgfqpoint{0.542cm}{1.435cm}}{\pgfqpoint{0.542cm}{1.399cm}}
\pgfpathcurveto{\pgfqpoint{0.542cm}{1.363cm}}{\pgfqpoint{0.557cm}{1.328cm}}{\pgfqpoint{0.582cm}{1.302cm}}
\pgfpathcurveto{\pgfqpoint{0.608cm}{1.276cm}}{\pgfqpoint{0.643cm}{1.262cm}}{\pgfqpoint{0.679cm}{1.262cm}}
\pgfpathcurveto{\pgfqpoint{0.715cm}{1.262cm}}{\pgfqpoint{0.75cm}{1.276cm}}{\pgfqpoint{0.775cm}{1.302cm}}
\pgfpathcurveto{\pgfqpoint{0.801cm}{1.328cm}}{\pgfqpoint{0.815cm}{1.363cm}}{\pgfqpoint{0.815cm}{1.399cm}}
\pgfusepath{fill}
\pgfpathmoveto{\pgfqpoint{1.345cm}{1.371cm}}
\pgfpathcurveto{\pgfqpoint{1.345cm}{1.408cm}}{\pgfqpoint{1.331cm}{1.442cm}}{\pgfqpoint{1.305cm}{1.468cm}}
\pgfpathcurveto{\pgfqpoint{1.28cm}{1.494cm}}{\pgfqpoint{1.245cm}{1.508cm}}{\pgfqpoint{1.209cm}{1.508cm}}
\pgfpathcurveto{\pgfqpoint{1.172cm}{1.508cm}}{\pgfqpoint{1.138cm}{1.494cm}}{\pgfqpoint{1.112cm}{1.468cm}}
\pgfpathcurveto{\pgfqpoint{1.087cm}{1.442cm}}{\pgfqpoint{1.072cm}{1.408cm}}{\pgfqpoint{1.072cm}{1.371cm}}
\pgfpathcurveto{\pgfqpoint{1.072cm}{1.335cm}}{\pgfqpoint{1.087cm}{1.3cm}}{\pgfqpoint{1.112cm}{1.274cm}}
\pgfpathcurveto{\pgfqpoint{1.138cm}{1.249cm}}{\pgfqpoint{1.172cm}{1.234cm}}{\pgfqpoint{1.209cm}{1.234cm}}
\pgfpathcurveto{\pgfqpoint{1.245cm}{1.234cm}}{\pgfqpoint{1.28cm}{1.249cm}}{\pgfqpoint{1.305cm}{1.274cm}}
\pgfpathcurveto{\pgfqpoint{1.331cm}{1.3cm}}{\pgfqpoint{1.345cm}{1.335cm}}{\pgfqpoint{1.345cm}{1.371cm}}
\pgfusepath{fill}
\begin{pgfscope}
\pgfsetdash{}{0cm}
\pgfsetlinewidth{0.818mm}
\pgfsetroundcap
\pgfsetmiterlimit{4.0}
\pgfpathmoveto{\pgfqpoint{0.682cm}{0.671cm}}
\pgfpathlineto{\pgfqpoint{0.682cm}{0.042cm}}
\pgfusepath{stroke}
\end{pgfscope}
\end{pgfscope}
\end{pgfscope}
\end{pgfscope}
\end{tikzpicture}}}_{\varepsilon} \circ \llbracket X_{\varepsilon}^2 \rrbracket -  b_{\varepsilon}X_{\varepsilon}, \]
   where
   $$
   b_{\varepsilon} := 3 \mathbbm{E} [(X^{\!\resizebox{0.6em}{!}{
\begin{tikzpicture}
\pgfpathmoveto{\pgfqpoint{0cm}{0cm}}
\pgfpathlineto{\pgfqpoint{1.376cm}{0cm}}
\pgfpathlineto{\pgfqpoint{1.376cm}{1.588cm}}
\pgfpathlineto{\pgfqpoint{0cm}{1.588cm}}
\pgfpathclose
\pgfusepath{clip}
\begin{pgfscope}
\begin{pgfscope}
\pgfpathmoveto{\pgfqpoint{0cm}{0cm}}
\pgfpathlineto{\pgfqpoint{1.376cm}{0cm}}
\pgfpathlineto{\pgfqpoint{1.376cm}{1.588cm}}
\pgfpathlineto{\pgfqpoint{0cm}{1.588cm}}
\pgfpathclose
\pgfusepath{clip}
\begin{pgfscope}
\begin{pgfscope}
\definecolor{eps2pgf_color}{gray}{0.976471}\pgfsetstrokecolor{eps2pgf_color}\pgfsetfillcolor{eps2pgf_color}
\pgfpathmoveto{\pgfqpoint{0cm}{0cm}}
\pgfpathlineto{\pgfqpoint{1.376cm}{0cm}}
\pgfpathlineto{\pgfqpoint{1.376cm}{1.588cm}}
\pgfpathlineto{\pgfqpoint{0cm}{1.588cm}}
\pgfpathclose
\pgfusepath{fill}
\end{pgfscope}
\begin{pgfscope}
\pgfsetdash{}{0cm}
\pgfsetlinewidth{0.818mm}
\pgfsetroundcap
\pgfsetroundjoin
\pgfsetmiterlimit{7.0}
\definecolor{eps2pgf_color}{gray}{0}\pgfsetstrokecolor{eps2pgf_color}\pgfsetfillcolor{eps2pgf_color}
\pgfpathmoveto{\pgfqpoint{0.117cm}{1.476cm}}
\pgfpathlineto{\pgfqpoint{0.682cm}{0.726cm}}
\pgfpathlineto{\pgfqpoint{1.246cm}{1.476cm}}
\pgfusepath{stroke}
\end{pgfscope}
\definecolor{eps2pgf_color}{gray}{0}\pgfsetstrokecolor{eps2pgf_color}\pgfsetfillcolor{eps2pgf_color}
\pgfpathmoveto{\pgfqpoint{0.273cm}{1.451cm}}
\pgfpathcurveto{\pgfqpoint{0.273cm}{1.487cm}}{\pgfqpoint{0.259cm}{1.522cm}}{\pgfqpoint{0.233cm}{1.547cm}}
\pgfpathcurveto{\pgfqpoint{0.207cm}{1.573cm}}{\pgfqpoint{0.173cm}{1.588cm}}{\pgfqpoint{0.137cm}{1.588cm}}
\pgfpathcurveto{\pgfqpoint{0.1cm}{1.588cm}}{\pgfqpoint{0.066cm}{1.573cm}}{\pgfqpoint{0.04cm}{1.547cm}}
\pgfpathcurveto{\pgfqpoint{0.014cm}{1.522cm}}{\pgfqpoint{0cm}{1.487cm}}{\pgfqpoint{0cm}{1.451cm}}
\pgfpathcurveto{\pgfqpoint{0cm}{1.414cm}}{\pgfqpoint{0.014cm}{1.379cm}}{\pgfqpoint{0.04cm}{1.354cm}}
\pgfpathcurveto{\pgfqpoint{0.066cm}{1.328cm}}{\pgfqpoint{0.1cm}{1.314cm}}{\pgfqpoint{0.137cm}{1.314cm}}
\pgfpathcurveto{\pgfqpoint{0.173cm}{1.314cm}}{\pgfqpoint{0.207cm}{1.328cm}}{\pgfqpoint{0.233cm}{1.354cm}}
\pgfpathcurveto{\pgfqpoint{0.259cm}{1.379cm}}{\pgfqpoint{0.273cm}{1.414cm}}{\pgfqpoint{0.273cm}{1.451cm}}
\pgfusepath{fill}
\pgfpathmoveto{\pgfqpoint{1.345cm}{1.426cm}}
\pgfpathcurveto{\pgfqpoint{1.345cm}{1.463cm}}{\pgfqpoint{1.331cm}{1.497cm}}{\pgfqpoint{1.305cm}{1.523cm}}
\pgfpathcurveto{\pgfqpoint{1.28cm}{1.549cm}}{\pgfqpoint{1.245cm}{1.563cm}}{\pgfqpoint{1.209cm}{1.563cm}}
\pgfpathcurveto{\pgfqpoint{1.172cm}{1.563cm}}{\pgfqpoint{1.138cm}{1.549cm}}{\pgfqpoint{1.112cm}{1.523cm}}
\pgfpathcurveto{\pgfqpoint{1.087cm}{1.497cm}}{\pgfqpoint{1.072cm}{1.463cm}}{\pgfqpoint{1.072cm}{1.426cm}}
\pgfpathcurveto{\pgfqpoint{1.072cm}{1.39cm}}{\pgfqpoint{1.087cm}{1.355cm}}{\pgfqpoint{1.112cm}{1.329cm}}
\pgfpathcurveto{\pgfqpoint{1.138cm}{1.304cm}}{\pgfqpoint{1.172cm}{1.289cm}}{\pgfqpoint{1.209cm}{1.289cm}}
\pgfpathcurveto{\pgfqpoint{1.245cm}{1.289cm}}{\pgfqpoint{1.28cm}{1.304cm}}{\pgfqpoint{1.305cm}{1.329cm}}
\pgfpathcurveto{\pgfqpoint{1.331cm}{1.355cm}}{\pgfqpoint{1.345cm}{1.39cm}}{\pgfqpoint{1.345cm}{1.426cm}}
\pgfusepath{fill}
\begin{pgfscope}
\pgfsetdash{}{0cm}
\pgfsetlinewidth{0.818mm}
\pgfsetroundcap
\pgfsetmiterlimit{4.0}
\pgfpathmoveto{\pgfqpoint{0.682cm}{0.726cm}}
\pgfpathlineto{\pgfqpoint{0.682cm}{0.097cm}}
\pgfusepath{stroke}
\end{pgfscope}
\end{pgfscope}
\end{pgfscope}
\end{pgfscope}
\end{tikzpicture}}}_{\varepsilon} \circ \llbracket X_{\varepsilon}^2 \rrbracket)(0)].
   $$
 Similarly, we define the periodic analogs.
 
  \renewcommand{\arraystretch}{1.5}
   \begin{table}
      \begin{center}
  \begin{tabular}{| r | c | c | c | c|c|c|c|c|}
  \hline
  $\tau$ & $X$ & $ \llbracket X^2 \rrbracket$ & $ \llbracket X^3 \rrbracket$ & $X^{\!\resizebox{0.6em}{!}{
\begin{tikzpicture}
\pgfpathmoveto{\pgfqpoint{0cm}{-0.035cm}}
\pgfpathlineto{\pgfqpoint{1.376cm}{-0.035cm}}
\pgfpathlineto{\pgfqpoint{1.376cm}{1.552cm}}
\pgfpathlineto{\pgfqpoint{0cm}{1.552cm}}
\pgfpathclose
\pgfusepath{clip}
\begin{pgfscope}
\begin{pgfscope}
\pgfpathmoveto{\pgfqpoint{0cm}{-0.035cm}}
\pgfpathlineto{\pgfqpoint{1.376cm}{-0.035cm}}
\pgfpathlineto{\pgfqpoint{1.376cm}{1.552cm}}
\pgfpathlineto{\pgfqpoint{0cm}{1.552cm}}
\pgfpathclose
\pgfusepath{clip}
\begin{pgfscope}
\begin{pgfscope}
\pgfsetdash{}{0cm}
\pgfsetlinewidth{0.818mm}
\pgfsetroundcap
\pgfsetroundjoin
\pgfsetmiterlimit{7.0}
\definecolor{eps2pgf_color}{gray}{0}\pgfsetstrokecolor{eps2pgf_color}\pgfsetfillcolor{eps2pgf_color}
\pgfpathmoveto{\pgfqpoint{0.117cm}{1.421cm}}
\pgfpathlineto{\pgfqpoint{0.682cm}{0.671cm}}
\pgfpathlineto{\pgfqpoint{1.246cm}{1.421cm}}
\pgfusepath{stroke}
\end{pgfscope}
\definecolor{eps2pgf_color}{gray}{0}\pgfsetstrokecolor{eps2pgf_color}\pgfsetfillcolor{eps2pgf_color}
\pgfpathmoveto{\pgfqpoint{0.273cm}{1.395cm}}
\pgfpathcurveto{\pgfqpoint{0.273cm}{1.432cm}}{\pgfqpoint{0.259cm}{1.467cm}}{\pgfqpoint{0.233cm}{1.492cm}}
\pgfpathcurveto{\pgfqpoint{0.207cm}{1.518cm}}{\pgfqpoint{0.173cm}{1.532cm}}{\pgfqpoint{0.137cm}{1.532cm}}
\pgfpathcurveto{\pgfqpoint{0.1cm}{1.532cm}}{\pgfqpoint{0.066cm}{1.518cm}}{\pgfqpoint{0.04cm}{1.492cm}}
\pgfpathcurveto{\pgfqpoint{0.014cm}{1.467cm}}{\pgfqpoint{0cm}{1.432cm}}{\pgfqpoint{0cm}{1.395cm}}
\pgfpathcurveto{\pgfqpoint{0cm}{1.359cm}}{\pgfqpoint{0.014cm}{1.324cm}}{\pgfqpoint{0.04cm}{1.299cm}}
\pgfpathcurveto{\pgfqpoint{0.066cm}{1.273cm}}{\pgfqpoint{0.1cm}{1.258cm}}{\pgfqpoint{0.137cm}{1.258cm}}
\pgfpathcurveto{\pgfqpoint{0.173cm}{1.258cm}}{\pgfqpoint{0.207cm}{1.273cm}}{\pgfqpoint{0.233cm}{1.299cm}}
\pgfpathcurveto{\pgfqpoint{0.259cm}{1.324cm}}{\pgfqpoint{0.273cm}{1.359cm}}{\pgfqpoint{0.273cm}{1.395cm}}
\pgfusepath{fill}
\begin{pgfscope}
\pgfsetdash{}{0cm}
\pgfsetlinewidth{0.818mm}
\pgfsetmiterlimit{7.0}
\pgfpathmoveto{\pgfqpoint{0.682cm}{0.671cm}}
\pgfpathlineto{\pgfqpoint{0.679cm}{1.418cm}}
\pgfusepath{stroke}
\end{pgfscope}
\pgfpathmoveto{\pgfqpoint{0.815cm}{1.399cm}}
\pgfpathcurveto{\pgfqpoint{0.815cm}{1.435cm}}{\pgfqpoint{0.801cm}{1.47cm}}{\pgfqpoint{0.775cm}{1.496cm}}
\pgfpathcurveto{\pgfqpoint{0.75cm}{1.521cm}}{\pgfqpoint{0.715cm}{1.536cm}}{\pgfqpoint{0.679cm}{1.536cm}}
\pgfpathcurveto{\pgfqpoint{0.643cm}{1.536cm}}{\pgfqpoint{0.608cm}{1.521cm}}{\pgfqpoint{0.582cm}{1.496cm}}
\pgfpathcurveto{\pgfqpoint{0.557cm}{1.47cm}}{\pgfqpoint{0.542cm}{1.435cm}}{\pgfqpoint{0.542cm}{1.399cm}}
\pgfpathcurveto{\pgfqpoint{0.542cm}{1.363cm}}{\pgfqpoint{0.557cm}{1.328cm}}{\pgfqpoint{0.582cm}{1.302cm}}
\pgfpathcurveto{\pgfqpoint{0.608cm}{1.276cm}}{\pgfqpoint{0.643cm}{1.262cm}}{\pgfqpoint{0.679cm}{1.262cm}}
\pgfpathcurveto{\pgfqpoint{0.715cm}{1.262cm}}{\pgfqpoint{0.75cm}{1.276cm}}{\pgfqpoint{0.775cm}{1.302cm}}
\pgfpathcurveto{\pgfqpoint{0.801cm}{1.328cm}}{\pgfqpoint{0.815cm}{1.363cm}}{\pgfqpoint{0.815cm}{1.399cm}}
\pgfusepath{fill}
\pgfpathmoveto{\pgfqpoint{1.345cm}{1.371cm}}
\pgfpathcurveto{\pgfqpoint{1.345cm}{1.408cm}}{\pgfqpoint{1.331cm}{1.442cm}}{\pgfqpoint{1.305cm}{1.468cm}}
\pgfpathcurveto{\pgfqpoint{1.28cm}{1.494cm}}{\pgfqpoint{1.245cm}{1.508cm}}{\pgfqpoint{1.209cm}{1.508cm}}
\pgfpathcurveto{\pgfqpoint{1.172cm}{1.508cm}}{\pgfqpoint{1.138cm}{1.494cm}}{\pgfqpoint{1.112cm}{1.468cm}}
\pgfpathcurveto{\pgfqpoint{1.087cm}{1.442cm}}{\pgfqpoint{1.072cm}{1.408cm}}{\pgfqpoint{1.072cm}{1.371cm}}
\pgfpathcurveto{\pgfqpoint{1.072cm}{1.335cm}}{\pgfqpoint{1.087cm}{1.3cm}}{\pgfqpoint{1.112cm}{1.274cm}}
\pgfpathcurveto{\pgfqpoint{1.138cm}{1.249cm}}{\pgfqpoint{1.172cm}{1.234cm}}{\pgfqpoint{1.209cm}{1.234cm}}
\pgfpathcurveto{\pgfqpoint{1.245cm}{1.234cm}}{\pgfqpoint{1.28cm}{1.249cm}}{\pgfqpoint{1.305cm}{1.274cm}}
\pgfpathcurveto{\pgfqpoint{1.331cm}{1.3cm}}{\pgfqpoint{1.345cm}{1.335cm}}{\pgfqpoint{1.345cm}{1.371cm}}
\pgfusepath{fill}
\begin{pgfscope}
\pgfsetdash{}{0cm}
\pgfsetlinewidth{0.818mm}
\pgfsetroundcap
\pgfsetmiterlimit{4.0}
\pgfpathmoveto{\pgfqpoint{0.682cm}{0.671cm}}
\pgfpathlineto{\pgfqpoint{0.682cm}{0.042cm}}
\pgfusepath{stroke}
\end{pgfscope}
\end{pgfscope}
\end{pgfscope}
\end{pgfscope}
\end{tikzpicture}}}$ & $X^{\!\resizebox{0.6em}{!}{
\begin{tikzpicture}
\pgfpathmoveto{\pgfqpoint{0cm}{0cm}}
\pgfpathlineto{\pgfqpoint{1.376cm}{0cm}}
\pgfpathlineto{\pgfqpoint{1.376cm}{1.588cm}}
\pgfpathlineto{\pgfqpoint{0cm}{1.588cm}}
\pgfpathclose
\pgfusepath{clip}
\begin{pgfscope}
\begin{pgfscope}
\pgfpathmoveto{\pgfqpoint{0cm}{0cm}}
\pgfpathlineto{\pgfqpoint{1.376cm}{0cm}}
\pgfpathlineto{\pgfqpoint{1.376cm}{1.588cm}}
\pgfpathlineto{\pgfqpoint{0cm}{1.588cm}}
\pgfpathclose
\pgfusepath{clip}
\begin{pgfscope}
\begin{pgfscope}
\definecolor{eps2pgf_color}{gray}{0.976471}\pgfsetstrokecolor{eps2pgf_color}\pgfsetfillcolor{eps2pgf_color}
\pgfpathmoveto{\pgfqpoint{0cm}{0cm}}
\pgfpathlineto{\pgfqpoint{1.376cm}{0cm}}
\pgfpathlineto{\pgfqpoint{1.376cm}{1.588cm}}
\pgfpathlineto{\pgfqpoint{0cm}{1.588cm}}
\pgfpathclose
\pgfusepath{fill}
\end{pgfscope}
\begin{pgfscope}
\pgfsetdash{}{0cm}
\pgfsetlinewidth{0.818mm}
\pgfsetroundcap
\pgfsetroundjoin
\pgfsetmiterlimit{7.0}
\definecolor{eps2pgf_color}{gray}{0}\pgfsetstrokecolor{eps2pgf_color}\pgfsetfillcolor{eps2pgf_color}
\pgfpathmoveto{\pgfqpoint{0.117cm}{1.476cm}}
\pgfpathlineto{\pgfqpoint{0.682cm}{0.726cm}}
\pgfpathlineto{\pgfqpoint{1.246cm}{1.476cm}}
\pgfusepath{stroke}
\end{pgfscope}
\definecolor{eps2pgf_color}{gray}{0}\pgfsetstrokecolor{eps2pgf_color}\pgfsetfillcolor{eps2pgf_color}
\pgfpathmoveto{\pgfqpoint{0.273cm}{1.451cm}}
\pgfpathcurveto{\pgfqpoint{0.273cm}{1.487cm}}{\pgfqpoint{0.259cm}{1.522cm}}{\pgfqpoint{0.233cm}{1.547cm}}
\pgfpathcurveto{\pgfqpoint{0.207cm}{1.573cm}}{\pgfqpoint{0.173cm}{1.588cm}}{\pgfqpoint{0.137cm}{1.588cm}}
\pgfpathcurveto{\pgfqpoint{0.1cm}{1.588cm}}{\pgfqpoint{0.066cm}{1.573cm}}{\pgfqpoint{0.04cm}{1.547cm}}
\pgfpathcurveto{\pgfqpoint{0.014cm}{1.522cm}}{\pgfqpoint{0cm}{1.487cm}}{\pgfqpoint{0cm}{1.451cm}}
\pgfpathcurveto{\pgfqpoint{0cm}{1.414cm}}{\pgfqpoint{0.014cm}{1.379cm}}{\pgfqpoint{0.04cm}{1.354cm}}
\pgfpathcurveto{\pgfqpoint{0.066cm}{1.328cm}}{\pgfqpoint{0.1cm}{1.314cm}}{\pgfqpoint{0.137cm}{1.314cm}}
\pgfpathcurveto{\pgfqpoint{0.173cm}{1.314cm}}{\pgfqpoint{0.207cm}{1.328cm}}{\pgfqpoint{0.233cm}{1.354cm}}
\pgfpathcurveto{\pgfqpoint{0.259cm}{1.379cm}}{\pgfqpoint{0.273cm}{1.414cm}}{\pgfqpoint{0.273cm}{1.451cm}}
\pgfusepath{fill}
\pgfpathmoveto{\pgfqpoint{1.345cm}{1.426cm}}
\pgfpathcurveto{\pgfqpoint{1.345cm}{1.463cm}}{\pgfqpoint{1.331cm}{1.497cm}}{\pgfqpoint{1.305cm}{1.523cm}}
\pgfpathcurveto{\pgfqpoint{1.28cm}{1.549cm}}{\pgfqpoint{1.245cm}{1.563cm}}{\pgfqpoint{1.209cm}{1.563cm}}
\pgfpathcurveto{\pgfqpoint{1.172cm}{1.563cm}}{\pgfqpoint{1.138cm}{1.549cm}}{\pgfqpoint{1.112cm}{1.523cm}}
\pgfpathcurveto{\pgfqpoint{1.087cm}{1.497cm}}{\pgfqpoint{1.072cm}{1.463cm}}{\pgfqpoint{1.072cm}{1.426cm}}
\pgfpathcurveto{\pgfqpoint{1.072cm}{1.39cm}}{\pgfqpoint{1.087cm}{1.355cm}}{\pgfqpoint{1.112cm}{1.329cm}}
\pgfpathcurveto{\pgfqpoint{1.138cm}{1.304cm}}{\pgfqpoint{1.172cm}{1.289cm}}{\pgfqpoint{1.209cm}{1.289cm}}
\pgfpathcurveto{\pgfqpoint{1.245cm}{1.289cm}}{\pgfqpoint{1.28cm}{1.304cm}}{\pgfqpoint{1.305cm}{1.329cm}}
\pgfpathcurveto{\pgfqpoint{1.331cm}{1.355cm}}{\pgfqpoint{1.345cm}{1.39cm}}{\pgfqpoint{1.345cm}{1.426cm}}
\pgfusepath{fill}
\begin{pgfscope}
\pgfsetdash{}{0cm}
\pgfsetlinewidth{0.818mm}
\pgfsetroundcap
\pgfsetmiterlimit{4.0}
\pgfpathmoveto{\pgfqpoint{0.682cm}{0.726cm}}
\pgfpathlineto{\pgfqpoint{0.682cm}{0.097cm}}
\pgfusepath{stroke}
\end{pgfscope}
\end{pgfscope}
\end{pgfscope}
\end{pgfscope}
\end{tikzpicture}}}$ & $X^{\!\resizebox{!}{.8em}{
\begin{tikzpicture}
\pgfpathmoveto{\pgfqpoint{0cm}{-0.035cm}}
\pgfpathlineto{\pgfqpoint{1.976cm}{-0.035cm}}
\pgfpathlineto{\pgfqpoint{1.976cm}{1.94cm}}
\pgfpathlineto{\pgfqpoint{0cm}{1.94cm}}
\pgfpathclose
\pgfusepath{clip}
\begin{pgfscope}
\begin{pgfscope}
\pgfpathmoveto{\pgfqpoint{0cm}{-0.035cm}}
\pgfpathlineto{\pgfqpoint{1.976cm}{-0.035cm}}
\pgfpathlineto{\pgfqpoint{1.976cm}{1.94cm}}
\pgfpathlineto{\pgfqpoint{0cm}{1.94cm}}
\pgfpathclose
\pgfusepath{clip}
\begin{pgfscope}
\begin{pgfscope}
\pgfsetdash{}{0cm}
\pgfsetlinewidth{0.818mm}
\pgfsetroundcap
\pgfsetroundjoin
\pgfsetmiterlimit{7.0}
\definecolor{eps2pgf_color}{gray}{0}\pgfsetstrokecolor{eps2pgf_color}\pgfsetfillcolor{eps2pgf_color}
\pgfpathmoveto{\pgfqpoint{0.117cm}{1.815cm}}
\pgfpathlineto{\pgfqpoint{0.682cm}{1.065cm}}
\pgfpathlineto{\pgfqpoint{1.246cm}{1.815cm}}
\pgfusepath{stroke}
\end{pgfscope}
\definecolor{eps2pgf_color}{gray}{0}\pgfsetstrokecolor{eps2pgf_color}\pgfsetfillcolor{eps2pgf_color}
\pgfpathmoveto{\pgfqpoint{0.273cm}{1.789cm}}
\pgfpathcurveto{\pgfqpoint{0.273cm}{1.825cm}}{\pgfqpoint{0.259cm}{1.86cm}}{\pgfqpoint{0.233cm}{1.886cm}}
\pgfpathcurveto{\pgfqpoint{0.207cm}{1.912cm}}{\pgfqpoint{0.173cm}{1.926cm}}{\pgfqpoint{0.137cm}{1.926cm}}
\pgfpathcurveto{\pgfqpoint{0.1cm}{1.926cm}}{\pgfqpoint{0.066cm}{1.912cm}}{\pgfqpoint{0.04cm}{1.886cm}}
\pgfpathcurveto{\pgfqpoint{0.014cm}{1.86cm}}{\pgfqpoint{0cm}{1.825cm}}{\pgfqpoint{0cm}{1.789cm}}
\pgfpathcurveto{\pgfqpoint{0cm}{1.753cm}}{\pgfqpoint{0.014cm}{1.718cm}}{\pgfqpoint{0.04cm}{1.692cm}}
\pgfpathcurveto{\pgfqpoint{0.066cm}{1.667cm}}{\pgfqpoint{0.1cm}{1.652cm}}{\pgfqpoint{0.137cm}{1.652cm}}
\pgfpathcurveto{\pgfqpoint{0.173cm}{1.652cm}}{\pgfqpoint{0.207cm}{1.667cm}}{\pgfqpoint{0.233cm}{1.692cm}}
\pgfpathcurveto{\pgfqpoint{0.259cm}{1.718cm}}{\pgfqpoint{0.273cm}{1.753cm}}{\pgfqpoint{0.273cm}{1.789cm}}
\pgfusepath{fill}
\begin{pgfscope}
\pgfsetdash{}{0cm}
\pgfsetlinewidth{0.818mm}
\pgfsetmiterlimit{7.0}
\pgfpathmoveto{\pgfqpoint{0.682cm}{1.065cm}}
\pgfpathlineto{\pgfqpoint{0.679cm}{1.812cm}}
\pgfusepath{stroke}
\end{pgfscope}
\pgfpathmoveto{\pgfqpoint{0.815cm}{1.793cm}}
\pgfpathcurveto{\pgfqpoint{0.815cm}{1.829cm}}{\pgfqpoint{0.801cm}{1.864cm}}{\pgfqpoint{0.775cm}{1.89cm}}
\pgfpathcurveto{\pgfqpoint{0.75cm}{1.915cm}}{\pgfqpoint{0.715cm}{1.93cm}}{\pgfqpoint{0.679cm}{1.93cm}}
\pgfpathcurveto{\pgfqpoint{0.643cm}{1.93cm}}{\pgfqpoint{0.608cm}{1.915cm}}{\pgfqpoint{0.582cm}{1.89cm}}
\pgfpathcurveto{\pgfqpoint{0.557cm}{1.864cm}}{\pgfqpoint{0.542cm}{1.829cm}}{\pgfqpoint{0.542cm}{1.793cm}}
\pgfpathcurveto{\pgfqpoint{0.542cm}{1.756cm}}{\pgfqpoint{0.557cm}{1.722cm}}{\pgfqpoint{0.582cm}{1.696cm}}
\pgfpathcurveto{\pgfqpoint{0.608cm}{1.67cm}}{\pgfqpoint{0.643cm}{1.656cm}}{\pgfqpoint{0.679cm}{1.656cm}}
\pgfpathcurveto{\pgfqpoint{0.715cm}{1.656cm}}{\pgfqpoint{0.75cm}{1.67cm}}{\pgfqpoint{0.775cm}{1.696cm}}
\pgfpathcurveto{\pgfqpoint{0.801cm}{1.722cm}}{\pgfqpoint{0.815cm}{1.756cm}}{\pgfqpoint{0.815cm}{1.793cm}}
\pgfusepath{fill}
\pgfpathmoveto{\pgfqpoint{1.345cm}{1.765cm}}
\pgfpathcurveto{\pgfqpoint{1.345cm}{1.801cm}}{\pgfqpoint{1.331cm}{1.836cm}}{\pgfqpoint{1.305cm}{1.862cm}}
\pgfpathcurveto{\pgfqpoint{1.28cm}{1.887cm}}{\pgfqpoint{1.245cm}{1.902cm}}{\pgfqpoint{1.209cm}{1.902cm}}
\pgfpathcurveto{\pgfqpoint{1.172cm}{1.902cm}}{\pgfqpoint{1.138cm}{1.887cm}}{\pgfqpoint{1.112cm}{1.862cm}}
\pgfpathcurveto{\pgfqpoint{1.087cm}{1.836cm}}{\pgfqpoint{1.072cm}{1.801cm}}{\pgfqpoint{1.072cm}{1.765cm}}
\pgfpathcurveto{\pgfqpoint{1.072cm}{1.728cm}}{\pgfqpoint{1.087cm}{1.694cm}}{\pgfqpoint{1.112cm}{1.668cm}}
\pgfpathcurveto{\pgfqpoint{1.138cm}{1.642cm}}{\pgfqpoint{1.172cm}{1.628cm}}{\pgfqpoint{1.209cm}{1.628cm}}
\pgfpathcurveto{\pgfqpoint{1.245cm}{1.628cm}}{\pgfqpoint{1.28cm}{1.642cm}}{\pgfqpoint{1.305cm}{1.668cm}}
\pgfpathcurveto{\pgfqpoint{1.331cm}{1.694cm}}{\pgfqpoint{1.345cm}{1.728cm}}{\pgfqpoint{1.345cm}{1.765cm}}
\pgfusepath{fill}
\begin{pgfscope}
\pgfsetdash{}{0cm}
\pgfsetlinewidth{0.818mm}
\pgfsetroundcap
\pgfsetroundjoin
\pgfsetmiterlimit{7.0}
\pgfpathmoveto{\pgfqpoint{0.682cm}{1.065cm}}
\pgfpathlineto{\pgfqpoint{1.246cm}{0.315cm}}
\pgfpathlineto{\pgfqpoint{1.811cm}{1.065cm}}
\pgfusepath{stroke}
\end{pgfscope}
\pgfpathmoveto{\pgfqpoint{1.948cm}{1.065cm}}
\pgfpathcurveto{\pgfqpoint{1.948cm}{1.101cm}}{\pgfqpoint{1.933cm}{1.136cm}}{\pgfqpoint{1.907cm}{1.162cm}}
\pgfpathcurveto{\pgfqpoint{1.882cm}{1.187cm}}{\pgfqpoint{1.847cm}{1.202cm}}{\pgfqpoint{1.811cm}{1.202cm}}
\pgfpathcurveto{\pgfqpoint{1.775cm}{1.202cm}}{\pgfqpoint{1.74cm}{1.187cm}}{\pgfqpoint{1.714cm}{1.162cm}}
\pgfpathcurveto{\pgfqpoint{1.689cm}{1.136cm}}{\pgfqpoint{1.674cm}{1.101cm}}{\pgfqpoint{1.674cm}{1.065cm}}
\pgfpathcurveto{\pgfqpoint{1.674cm}{1.029cm}}{\pgfqpoint{1.689cm}{0.994cm}}{\pgfqpoint{1.714cm}{0.968cm}}
\pgfpathcurveto{\pgfqpoint{1.74cm}{0.942cm}}{\pgfqpoint{1.775cm}{0.928cm}}{\pgfqpoint{1.811cm}{0.928cm}}
\pgfpathcurveto{\pgfqpoint{1.847cm}{0.928cm}}{\pgfqpoint{1.882cm}{0.942cm}}{\pgfqpoint{1.907cm}{0.968cm}}
\pgfpathcurveto{\pgfqpoint{1.933cm}{0.994cm}}{\pgfqpoint{1.948cm}{1.029cm}}{\pgfqpoint{1.948cm}{1.065cm}}
\pgfusepath{fill}
\begin{pgfscope}
\pgfsetdash{}{0cm}
\pgfsetlinewidth{0.818mm}
\pgfsetmiterlimit{4.0}
\pgfpathmoveto{\pgfqpoint{1.383cm}{0.178cm}}
\pgfpathcurveto{\pgfqpoint{1.383cm}{0.214cm}}{\pgfqpoint{1.369cm}{0.249cm}}{\pgfqpoint{1.343cm}{0.275cm}}
\pgfpathcurveto{\pgfqpoint{1.317cm}{0.3cm}}{\pgfqpoint{1.283cm}{0.315cm}}{\pgfqpoint{1.246cm}{0.315cm}}
\pgfpathcurveto{\pgfqpoint{1.21cm}{0.315cm}}{\pgfqpoint{1.175cm}{0.3cm}}{\pgfqpoint{1.15cm}{0.275cm}}
\pgfpathcurveto{\pgfqpoint{1.124cm}{0.249cm}}{\pgfqpoint{1.11cm}{0.214cm}}{\pgfqpoint{1.11cm}{0.178cm}}
\pgfpathcurveto{\pgfqpoint{1.11cm}{0.141cm}}{\pgfqpoint{1.124cm}{0.107cm}}{\pgfqpoint{1.15cm}{0.081cm}}
\pgfpathcurveto{\pgfqpoint{1.175cm}{0.055cm}}{\pgfqpoint{1.21cm}{0.041cm}}{\pgfqpoint{1.246cm}{0.041cm}}
\pgfpathcurveto{\pgfqpoint{1.283cm}{0.041cm}}{\pgfqpoint{1.317cm}{0.055cm}}{\pgfqpoint{1.343cm}{0.081cm}}
\pgfpathcurveto{\pgfqpoint{1.369cm}{0.107cm}}{\pgfqpoint{1.383cm}{0.141cm}}{\pgfqpoint{1.383cm}{0.178cm}}
\pgfusepath{stroke}
\end{pgfscope}
\end{pgfscope}
\end{pgfscope}
\end{pgfscope}
\end{tikzpicture}}}$  & $X^{\!\resizebox{!}{.8em}{
\begin{tikzpicture}
\pgfpathmoveto{\pgfqpoint{0cm}{-0.035cm}}
\pgfpathlineto{\pgfqpoint{1.976cm}{-0.035cm}}
\pgfpathlineto{\pgfqpoint{1.976cm}{1.94cm}}
\pgfpathlineto{\pgfqpoint{0cm}{1.94cm}}
\pgfpathclose
\pgfusepath{clip}
\begin{pgfscope}
\begin{pgfscope}
\pgfpathmoveto{\pgfqpoint{0cm}{-0.035cm}}
\pgfpathlineto{\pgfqpoint{1.976cm}{-0.035cm}}
\pgfpathlineto{\pgfqpoint{1.976cm}{1.94cm}}
\pgfpathlineto{\pgfqpoint{0cm}{1.94cm}}
\pgfpathclose
\pgfusepath{clip}
\begin{pgfscope}
\begin{pgfscope}
\pgfsetdash{}{0cm}
\pgfsetlinewidth{0.818mm}
\pgfsetroundcap
\pgfsetroundjoin
\pgfsetmiterlimit{7.0}
\definecolor{eps2pgf_color}{gray}{0}\pgfsetstrokecolor{eps2pgf_color}\pgfsetfillcolor{eps2pgf_color}
\pgfpathmoveto{\pgfqpoint{0.117cm}{1.815cm}}
\pgfpathlineto{\pgfqpoint{0.682cm}{1.065cm}}
\pgfpathlineto{\pgfqpoint{1.246cm}{1.815cm}}
\pgfusepath{stroke}
\end{pgfscope}
\definecolor{eps2pgf_color}{gray}{0}\pgfsetstrokecolor{eps2pgf_color}\pgfsetfillcolor{eps2pgf_color}
\pgfpathmoveto{\pgfqpoint{0.273cm}{1.789cm}}
\pgfpathcurveto{\pgfqpoint{0.273cm}{1.825cm}}{\pgfqpoint{0.259cm}{1.86cm}}{\pgfqpoint{0.233cm}{1.886cm}}
\pgfpathcurveto{\pgfqpoint{0.207cm}{1.912cm}}{\pgfqpoint{0.173cm}{1.926cm}}{\pgfqpoint{0.137cm}{1.926cm}}
\pgfpathcurveto{\pgfqpoint{0.1cm}{1.926cm}}{\pgfqpoint{0.066cm}{1.912cm}}{\pgfqpoint{0.04cm}{1.886cm}}
\pgfpathcurveto{\pgfqpoint{0.014cm}{1.86cm}}{\pgfqpoint{0cm}{1.825cm}}{\pgfqpoint{0cm}{1.789cm}}
\pgfpathcurveto{\pgfqpoint{0cm}{1.753cm}}{\pgfqpoint{0.014cm}{1.718cm}}{\pgfqpoint{0.04cm}{1.692cm}}
\pgfpathcurveto{\pgfqpoint{0.066cm}{1.667cm}}{\pgfqpoint{0.1cm}{1.652cm}}{\pgfqpoint{0.137cm}{1.652cm}}
\pgfpathcurveto{\pgfqpoint{0.173cm}{1.652cm}}{\pgfqpoint{0.207cm}{1.667cm}}{\pgfqpoint{0.233cm}{1.692cm}}
\pgfpathcurveto{\pgfqpoint{0.259cm}{1.718cm}}{\pgfqpoint{0.273cm}{1.753cm}}{\pgfqpoint{0.273cm}{1.789cm}}
\pgfusepath{fill}
\pgfpathmoveto{\pgfqpoint{1.345cm}{1.765cm}}
\pgfpathcurveto{\pgfqpoint{1.345cm}{1.801cm}}{\pgfqpoint{1.331cm}{1.836cm}}{\pgfqpoint{1.305cm}{1.862cm}}
\pgfpathcurveto{\pgfqpoint{1.28cm}{1.887cm}}{\pgfqpoint{1.245cm}{1.902cm}}{\pgfqpoint{1.209cm}{1.902cm}}
\pgfpathcurveto{\pgfqpoint{1.172cm}{1.902cm}}{\pgfqpoint{1.138cm}{1.887cm}}{\pgfqpoint{1.112cm}{1.862cm}}
\pgfpathcurveto{\pgfqpoint{1.087cm}{1.836cm}}{\pgfqpoint{1.072cm}{1.801cm}}{\pgfqpoint{1.072cm}{1.765cm}}
\pgfpathcurveto{\pgfqpoint{1.072cm}{1.728cm}}{\pgfqpoint{1.087cm}{1.694cm}}{\pgfqpoint{1.112cm}{1.668cm}}
\pgfpathcurveto{\pgfqpoint{1.138cm}{1.642cm}}{\pgfqpoint{1.172cm}{1.628cm}}{\pgfqpoint{1.209cm}{1.628cm}}
\pgfpathcurveto{\pgfqpoint{1.245cm}{1.628cm}}{\pgfqpoint{1.28cm}{1.642cm}}{\pgfqpoint{1.305cm}{1.668cm}}
\pgfpathcurveto{\pgfqpoint{1.331cm}{1.694cm}}{\pgfqpoint{1.345cm}{1.728cm}}{\pgfqpoint{1.345cm}{1.765cm}}
\pgfusepath{fill}
\begin{pgfscope}
\pgfsetdash{}{0cm}
\pgfsetlinewidth{0.818mm}
\pgfsetroundcap
\pgfsetroundjoin
\pgfsetmiterlimit{7.0}
\pgfpathmoveto{\pgfqpoint{0.682cm}{1.065cm}}
\pgfpathlineto{\pgfqpoint{1.246cm}{0.315cm}}
\pgfpathlineto{\pgfqpoint{1.811cm}{1.065cm}}
\pgfusepath{stroke}
\end{pgfscope}
\pgfpathmoveto{\pgfqpoint{1.948cm}{1.065cm}}
\pgfpathcurveto{\pgfqpoint{1.948cm}{1.101cm}}{\pgfqpoint{1.933cm}{1.136cm}}{\pgfqpoint{1.907cm}{1.162cm}}
\pgfpathcurveto{\pgfqpoint{1.882cm}{1.187cm}}{\pgfqpoint{1.847cm}{1.202cm}}{\pgfqpoint{1.811cm}{1.202cm}}
\pgfpathcurveto{\pgfqpoint{1.775cm}{1.202cm}}{\pgfqpoint{1.74cm}{1.187cm}}{\pgfqpoint{1.714cm}{1.162cm}}
\pgfpathcurveto{\pgfqpoint{1.689cm}{1.136cm}}{\pgfqpoint{1.674cm}{1.101cm}}{\pgfqpoint{1.674cm}{1.065cm}}
\pgfpathcurveto{\pgfqpoint{1.674cm}{1.029cm}}{\pgfqpoint{1.689cm}{0.994cm}}{\pgfqpoint{1.714cm}{0.968cm}}
\pgfpathcurveto{\pgfqpoint{1.74cm}{0.942cm}}{\pgfqpoint{1.775cm}{0.928cm}}{\pgfqpoint{1.811cm}{0.928cm}}
\pgfpathcurveto{\pgfqpoint{1.847cm}{0.928cm}}{\pgfqpoint{1.882cm}{0.942cm}}{\pgfqpoint{1.907cm}{0.968cm}}
\pgfpathcurveto{\pgfqpoint{1.933cm}{0.994cm}}{\pgfqpoint{1.948cm}{1.029cm}}{\pgfqpoint{1.948cm}{1.065cm}}
\pgfusepath{fill}
\begin{pgfscope}
\pgfsetdash{}{0cm}
\pgfsetlinewidth{0.818mm}
\pgfsetmiterlimit{7.0}
\pgfpathmoveto{\pgfqpoint{1.246cm}{0.315cm}}
\pgfpathlineto{\pgfqpoint{1.244cm}{1.061cm}}
\pgfusepath{stroke}
\end{pgfscope}
\pgfpathmoveto{\pgfqpoint{1.38cm}{1.065cm}}
\pgfpathcurveto{\pgfqpoint{1.38cm}{1.101cm}}{\pgfqpoint{1.366cm}{1.136cm}}{\pgfqpoint{1.34cm}{1.162cm}}
\pgfpathcurveto{\pgfqpoint{1.315cm}{1.187cm}}{\pgfqpoint{1.28cm}{1.202cm}}{\pgfqpoint{1.244cm}{1.202cm}}
\pgfpathcurveto{\pgfqpoint{1.207cm}{1.202cm}}{\pgfqpoint{1.173cm}{1.187cm}}{\pgfqpoint{1.147cm}{1.162cm}}
\pgfpathcurveto{\pgfqpoint{1.121cm}{1.136cm}}{\pgfqpoint{1.107cm}{1.101cm}}{\pgfqpoint{1.107cm}{1.065cm}}
\pgfpathcurveto{\pgfqpoint{1.107cm}{1.029cm}}{\pgfqpoint{1.121cm}{0.994cm}}{\pgfqpoint{1.147cm}{0.968cm}}
\pgfpathcurveto{\pgfqpoint{1.173cm}{0.942cm}}{\pgfqpoint{1.207cm}{0.928cm}}{\pgfqpoint{1.244cm}{0.928cm}}
\pgfpathcurveto{\pgfqpoint{1.28cm}{0.928cm}}{\pgfqpoint{1.315cm}{0.942cm}}{\pgfqpoint{1.34cm}{0.968cm}}
\pgfpathcurveto{\pgfqpoint{1.366cm}{0.994cm}}{\pgfqpoint{1.38cm}{1.029cm}}{\pgfqpoint{1.38cm}{1.065cm}}
\pgfusepath{fill}
\begin{pgfscope}
\pgfsetdash{}{0cm}
\pgfsetlinewidth{0.818mm}
\pgfsetmiterlimit{4.0}
\pgfpathmoveto{\pgfqpoint{1.383cm}{0.178cm}}
\pgfpathcurveto{\pgfqpoint{1.383cm}{0.214cm}}{\pgfqpoint{1.369cm}{0.249cm}}{\pgfqpoint{1.343cm}{0.275cm}}
\pgfpathcurveto{\pgfqpoint{1.317cm}{0.3cm}}{\pgfqpoint{1.283cm}{0.315cm}}{\pgfqpoint{1.246cm}{0.315cm}}
\pgfpathcurveto{\pgfqpoint{1.21cm}{0.315cm}}{\pgfqpoint{1.175cm}{0.3cm}}{\pgfqpoint{1.15cm}{0.275cm}}
\pgfpathcurveto{\pgfqpoint{1.124cm}{0.249cm}}{\pgfqpoint{1.11cm}{0.214cm}}{\pgfqpoint{1.11cm}{0.178cm}}
\pgfpathcurveto{\pgfqpoint{1.11cm}{0.141cm}}{\pgfqpoint{1.124cm}{0.107cm}}{\pgfqpoint{1.15cm}{0.081cm}}
\pgfpathcurveto{\pgfqpoint{1.175cm}{0.055cm}}{\pgfqpoint{1.21cm}{0.041cm}}{\pgfqpoint{1.246cm}{0.041cm}}
\pgfpathcurveto{\pgfqpoint{1.283cm}{0.041cm}}{\pgfqpoint{1.317cm}{0.055cm}}{\pgfqpoint{1.343cm}{0.081cm}}
\pgfpathcurveto{\pgfqpoint{1.369cm}{0.107cm}}{\pgfqpoint{1.383cm}{0.141cm}}{\pgfqpoint{1.383cm}{0.178cm}}
\pgfusepath{stroke}
\end{pgfscope}
\end{pgfscope}
\end{pgfscope}
\end{pgfscope}
\end{tikzpicture}}}$ & $ X^{\!\resizebox{!}{.8em}{
\begin{tikzpicture}
\pgfpathmoveto{\pgfqpoint{0cm}{-0.035cm}}
\pgfpathlineto{\pgfqpoint{1.976cm}{-0.035cm}}
\pgfpathlineto{\pgfqpoint{1.976cm}{1.94cm}}
\pgfpathlineto{\pgfqpoint{0cm}{1.94cm}}
\pgfpathclose
\pgfusepath{clip}
\begin{pgfscope}
\begin{pgfscope}
\pgfpathmoveto{\pgfqpoint{0cm}{-0.035cm}}
\pgfpathlineto{\pgfqpoint{1.976cm}{-0.035cm}}
\pgfpathlineto{\pgfqpoint{1.976cm}{1.94cm}}
\pgfpathlineto{\pgfqpoint{0cm}{1.94cm}}
\pgfpathclose
\pgfusepath{clip}
\begin{pgfscope}
\begin{pgfscope}
\pgfsetdash{}{0cm}
\pgfsetlinewidth{0.818mm}
\pgfsetroundcap
\pgfsetroundjoin
\pgfsetmiterlimit{7.0}
\definecolor{eps2pgf_color}{gray}{0}\pgfsetstrokecolor{eps2pgf_color}\pgfsetfillcolor{eps2pgf_color}
\pgfpathmoveto{\pgfqpoint{0.117cm}{1.815cm}}
\pgfpathlineto{\pgfqpoint{0.682cm}{1.065cm}}
\pgfpathlineto{\pgfqpoint{1.246cm}{1.815cm}}
\pgfusepath{stroke}
\end{pgfscope}
\definecolor{eps2pgf_color}{gray}{0}\pgfsetstrokecolor{eps2pgf_color}\pgfsetfillcolor{eps2pgf_color}
\pgfpathmoveto{\pgfqpoint{0.273cm}{1.789cm}}
\pgfpathcurveto{\pgfqpoint{0.273cm}{1.825cm}}{\pgfqpoint{0.259cm}{1.86cm}}{\pgfqpoint{0.233cm}{1.886cm}}
\pgfpathcurveto{\pgfqpoint{0.207cm}{1.912cm}}{\pgfqpoint{0.173cm}{1.926cm}}{\pgfqpoint{0.137cm}{1.926cm}}
\pgfpathcurveto{\pgfqpoint{0.1cm}{1.926cm}}{\pgfqpoint{0.066cm}{1.912cm}}{\pgfqpoint{0.04cm}{1.886cm}}
\pgfpathcurveto{\pgfqpoint{0.014cm}{1.86cm}}{\pgfqpoint{0cm}{1.825cm}}{\pgfqpoint{0cm}{1.789cm}}
\pgfpathcurveto{\pgfqpoint{0cm}{1.753cm}}{\pgfqpoint{0.014cm}{1.718cm}}{\pgfqpoint{0.04cm}{1.692cm}}
\pgfpathcurveto{\pgfqpoint{0.066cm}{1.667cm}}{\pgfqpoint{0.1cm}{1.652cm}}{\pgfqpoint{0.137cm}{1.652cm}}
\pgfpathcurveto{\pgfqpoint{0.173cm}{1.652cm}}{\pgfqpoint{0.207cm}{1.667cm}}{\pgfqpoint{0.233cm}{1.692cm}}
\pgfpathcurveto{\pgfqpoint{0.259cm}{1.718cm}}{\pgfqpoint{0.273cm}{1.753cm}}{\pgfqpoint{0.273cm}{1.789cm}}
\pgfusepath{fill}
\begin{pgfscope}
\pgfsetdash{}{0cm}
\pgfsetlinewidth{0.818mm}
\pgfsetmiterlimit{7.0}
\pgfpathmoveto{\pgfqpoint{0.682cm}{1.065cm}}
\pgfpathlineto{\pgfqpoint{0.679cm}{1.812cm}}
\pgfusepath{stroke}
\end{pgfscope}
\pgfpathmoveto{\pgfqpoint{0.815cm}{1.793cm}}
\pgfpathcurveto{\pgfqpoint{0.815cm}{1.829cm}}{\pgfqpoint{0.801cm}{1.864cm}}{\pgfqpoint{0.775cm}{1.89cm}}
\pgfpathcurveto{\pgfqpoint{0.75cm}{1.915cm}}{\pgfqpoint{0.715cm}{1.93cm}}{\pgfqpoint{0.679cm}{1.93cm}}
\pgfpathcurveto{\pgfqpoint{0.643cm}{1.93cm}}{\pgfqpoint{0.608cm}{1.915cm}}{\pgfqpoint{0.582cm}{1.89cm}}
\pgfpathcurveto{\pgfqpoint{0.557cm}{1.864cm}}{\pgfqpoint{0.542cm}{1.829cm}}{\pgfqpoint{0.542cm}{1.793cm}}
\pgfpathcurveto{\pgfqpoint{0.542cm}{1.756cm}}{\pgfqpoint{0.557cm}{1.722cm}}{\pgfqpoint{0.582cm}{1.696cm}}
\pgfpathcurveto{\pgfqpoint{0.608cm}{1.67cm}}{\pgfqpoint{0.643cm}{1.656cm}}{\pgfqpoint{0.679cm}{1.656cm}}
\pgfpathcurveto{\pgfqpoint{0.715cm}{1.656cm}}{\pgfqpoint{0.75cm}{1.67cm}}{\pgfqpoint{0.775cm}{1.696cm}}
\pgfpathcurveto{\pgfqpoint{0.801cm}{1.722cm}}{\pgfqpoint{0.815cm}{1.756cm}}{\pgfqpoint{0.815cm}{1.793cm}}
\pgfusepath{fill}
\pgfpathmoveto{\pgfqpoint{1.345cm}{1.765cm}}
\pgfpathcurveto{\pgfqpoint{1.345cm}{1.801cm}}{\pgfqpoint{1.331cm}{1.836cm}}{\pgfqpoint{1.305cm}{1.862cm}}
\pgfpathcurveto{\pgfqpoint{1.28cm}{1.887cm}}{\pgfqpoint{1.245cm}{1.902cm}}{\pgfqpoint{1.209cm}{1.902cm}}
\pgfpathcurveto{\pgfqpoint{1.172cm}{1.902cm}}{\pgfqpoint{1.138cm}{1.887cm}}{\pgfqpoint{1.112cm}{1.862cm}}
\pgfpathcurveto{\pgfqpoint{1.087cm}{1.836cm}}{\pgfqpoint{1.072cm}{1.801cm}}{\pgfqpoint{1.072cm}{1.765cm}}
\pgfpathcurveto{\pgfqpoint{1.072cm}{1.728cm}}{\pgfqpoint{1.087cm}{1.694cm}}{\pgfqpoint{1.112cm}{1.668cm}}
\pgfpathcurveto{\pgfqpoint{1.138cm}{1.642cm}}{\pgfqpoint{1.172cm}{1.628cm}}{\pgfqpoint{1.209cm}{1.628cm}}
\pgfpathcurveto{\pgfqpoint{1.245cm}{1.628cm}}{\pgfqpoint{1.28cm}{1.642cm}}{\pgfqpoint{1.305cm}{1.668cm}}
\pgfpathcurveto{\pgfqpoint{1.331cm}{1.694cm}}{\pgfqpoint{1.345cm}{1.728cm}}{\pgfqpoint{1.345cm}{1.765cm}}
\pgfusepath{fill}
\begin{pgfscope}
\pgfsetdash{}{0cm}
\pgfsetlinewidth{0.818mm}
\pgfsetroundcap
\pgfsetroundjoin
\pgfsetmiterlimit{7.0}
\pgfpathmoveto{\pgfqpoint{0.682cm}{1.065cm}}
\pgfpathlineto{\pgfqpoint{1.246cm}{0.315cm}}
\pgfpathlineto{\pgfqpoint{1.811cm}{1.065cm}}
\pgfusepath{stroke}
\end{pgfscope}
\pgfpathmoveto{\pgfqpoint{1.948cm}{1.065cm}}
\pgfpathcurveto{\pgfqpoint{1.948cm}{1.101cm}}{\pgfqpoint{1.933cm}{1.136cm}}{\pgfqpoint{1.907cm}{1.162cm}}
\pgfpathcurveto{\pgfqpoint{1.882cm}{1.187cm}}{\pgfqpoint{1.847cm}{1.202cm}}{\pgfqpoint{1.811cm}{1.202cm}}
\pgfpathcurveto{\pgfqpoint{1.775cm}{1.202cm}}{\pgfqpoint{1.74cm}{1.187cm}}{\pgfqpoint{1.714cm}{1.162cm}}
\pgfpathcurveto{\pgfqpoint{1.689cm}{1.136cm}}{\pgfqpoint{1.674cm}{1.101cm}}{\pgfqpoint{1.674cm}{1.065cm}}
\pgfpathcurveto{\pgfqpoint{1.674cm}{1.029cm}}{\pgfqpoint{1.689cm}{0.994cm}}{\pgfqpoint{1.714cm}{0.968cm}}
\pgfpathcurveto{\pgfqpoint{1.74cm}{0.942cm}}{\pgfqpoint{1.775cm}{0.928cm}}{\pgfqpoint{1.811cm}{0.928cm}}
\pgfpathcurveto{\pgfqpoint{1.847cm}{0.928cm}}{\pgfqpoint{1.882cm}{0.942cm}}{\pgfqpoint{1.907cm}{0.968cm}}
\pgfpathcurveto{\pgfqpoint{1.933cm}{0.994cm}}{\pgfqpoint{1.948cm}{1.029cm}}{\pgfqpoint{1.948cm}{1.065cm}}
\pgfusepath{fill}
\begin{pgfscope}
\pgfsetdash{}{0cm}
\pgfsetlinewidth{0.818mm}
\pgfsetmiterlimit{7.0}
\pgfpathmoveto{\pgfqpoint{1.246cm}{0.315cm}}
\pgfpathlineto{\pgfqpoint{1.244cm}{1.061cm}}
\pgfusepath{stroke}
\end{pgfscope}
\pgfpathmoveto{\pgfqpoint{1.38cm}{1.065cm}}
\pgfpathcurveto{\pgfqpoint{1.38cm}{1.101cm}}{\pgfqpoint{1.366cm}{1.136cm}}{\pgfqpoint{1.34cm}{1.162cm}}
\pgfpathcurveto{\pgfqpoint{1.315cm}{1.187cm}}{\pgfqpoint{1.28cm}{1.202cm}}{\pgfqpoint{1.244cm}{1.202cm}}
\pgfpathcurveto{\pgfqpoint{1.207cm}{1.202cm}}{\pgfqpoint{1.173cm}{1.187cm}}{\pgfqpoint{1.147cm}{1.162cm}}
\pgfpathcurveto{\pgfqpoint{1.121cm}{1.136cm}}{\pgfqpoint{1.107cm}{1.101cm}}{\pgfqpoint{1.107cm}{1.065cm}}
\pgfpathcurveto{\pgfqpoint{1.107cm}{1.029cm}}{\pgfqpoint{1.121cm}{0.994cm}}{\pgfqpoint{1.147cm}{0.968cm}}
\pgfpathcurveto{\pgfqpoint{1.173cm}{0.942cm}}{\pgfqpoint{1.207cm}{0.928cm}}{\pgfqpoint{1.244cm}{0.928cm}}
\pgfpathcurveto{\pgfqpoint{1.28cm}{0.928cm}}{\pgfqpoint{1.315cm}{0.942cm}}{\pgfqpoint{1.34cm}{0.968cm}}
\pgfpathcurveto{\pgfqpoint{1.366cm}{0.994cm}}{\pgfqpoint{1.38cm}{1.029cm}}{\pgfqpoint{1.38cm}{1.065cm}}
\pgfusepath{fill}
\begin{pgfscope}
\pgfsetdash{}{0cm}
\pgfsetlinewidth{0.818mm}
\pgfsetmiterlimit{4.0}
\pgfpathmoveto{\pgfqpoint{1.383cm}{0.178cm}}
\pgfpathcurveto{\pgfqpoint{1.383cm}{0.214cm}}{\pgfqpoint{1.369cm}{0.249cm}}{\pgfqpoint{1.343cm}{0.275cm}}
\pgfpathcurveto{\pgfqpoint{1.317cm}{0.3cm}}{\pgfqpoint{1.283cm}{0.315cm}}{\pgfqpoint{1.246cm}{0.315cm}}
\pgfpathcurveto{\pgfqpoint{1.21cm}{0.315cm}}{\pgfqpoint{1.175cm}{0.3cm}}{\pgfqpoint{1.15cm}{0.275cm}}
\pgfpathcurveto{\pgfqpoint{1.124cm}{0.249cm}}{\pgfqpoint{1.11cm}{0.214cm}}{\pgfqpoint{1.11cm}{0.178cm}}
\pgfpathcurveto{\pgfqpoint{1.11cm}{0.141cm}}{\pgfqpoint{1.124cm}{0.107cm}}{\pgfqpoint{1.15cm}{0.081cm}}
\pgfpathcurveto{\pgfqpoint{1.175cm}{0.055cm}}{\pgfqpoint{1.21cm}{0.041cm}}{\pgfqpoint{1.246cm}{0.041cm}}
\pgfpathcurveto{\pgfqpoint{1.283cm}{0.041cm}}{\pgfqpoint{1.317cm}{0.055cm}}{\pgfqpoint{1.343cm}{0.081cm}}
\pgfpathcurveto{\pgfqpoint{1.369cm}{0.107cm}}{\pgfqpoint{1.383cm}{0.141cm}}{\pgfqpoint{1.383cm}{0.178cm}}
\pgfusepath{stroke}
\end{pgfscope}
\end{pgfscope}
\end{pgfscope}
\end{pgfscope}
\end{tikzpicture}}}$\\ \hline 
  $\alpha_{\tau}$ & $-\frac{1}{2}-\kappa$ & $-1-\kappa$ & $-\frac32-\kappa $& $\frac{1}{2}-\kappa$ & $1-\kappa$ & $-\kappa$ & $-\kappa$ & $-\frac{1}{2}-\kappa$ \\
  \hline
  \end{tabular}
  \caption{Space regularity of stochastic objects in the $d=5$ elliptic or $d=3$ parabolic case.}
  \label{t:reg}
  \end{center}
  \end{table}
 
 \begin{theorem}\label{thm:renorm2}
Let $d=5$. Let $\rho(x)=\langle x\rangle^{-\nu}$ for some $\nu>0$. Then there exist random distributions
\begin{equation}\label{eq:r45}
X,\llbracket X^{2}\rrbracket,\llbracket X^{3}\rrbracket, X^{\!\resizebox{0.6em}{!}{
\begin{tikzpicture}
\pgfpathmoveto{\pgfqpoint{0cm}{-0.035cm}}
\pgfpathlineto{\pgfqpoint{1.376cm}{-0.035cm}}
\pgfpathlineto{\pgfqpoint{1.376cm}{1.552cm}}
\pgfpathlineto{\pgfqpoint{0cm}{1.552cm}}
\pgfpathclose
\pgfusepath{clip}
\begin{pgfscope}
\begin{pgfscope}
\pgfpathmoveto{\pgfqpoint{0cm}{-0.035cm}}
\pgfpathlineto{\pgfqpoint{1.376cm}{-0.035cm}}
\pgfpathlineto{\pgfqpoint{1.376cm}{1.552cm}}
\pgfpathlineto{\pgfqpoint{0cm}{1.552cm}}
\pgfpathclose
\pgfusepath{clip}
\begin{pgfscope}
\begin{pgfscope}
\pgfsetdash{}{0cm}
\pgfsetlinewidth{0.818mm}
\pgfsetroundcap
\pgfsetroundjoin
\pgfsetmiterlimit{7.0}
\definecolor{eps2pgf_color}{gray}{0}\pgfsetstrokecolor{eps2pgf_color}\pgfsetfillcolor{eps2pgf_color}
\pgfpathmoveto{\pgfqpoint{0.117cm}{1.421cm}}
\pgfpathlineto{\pgfqpoint{0.682cm}{0.671cm}}
\pgfpathlineto{\pgfqpoint{1.246cm}{1.421cm}}
\pgfusepath{stroke}
\end{pgfscope}
\definecolor{eps2pgf_color}{gray}{0}\pgfsetstrokecolor{eps2pgf_color}\pgfsetfillcolor{eps2pgf_color}
\pgfpathmoveto{\pgfqpoint{0.273cm}{1.395cm}}
\pgfpathcurveto{\pgfqpoint{0.273cm}{1.432cm}}{\pgfqpoint{0.259cm}{1.467cm}}{\pgfqpoint{0.233cm}{1.492cm}}
\pgfpathcurveto{\pgfqpoint{0.207cm}{1.518cm}}{\pgfqpoint{0.173cm}{1.532cm}}{\pgfqpoint{0.137cm}{1.532cm}}
\pgfpathcurveto{\pgfqpoint{0.1cm}{1.532cm}}{\pgfqpoint{0.066cm}{1.518cm}}{\pgfqpoint{0.04cm}{1.492cm}}
\pgfpathcurveto{\pgfqpoint{0.014cm}{1.467cm}}{\pgfqpoint{0cm}{1.432cm}}{\pgfqpoint{0cm}{1.395cm}}
\pgfpathcurveto{\pgfqpoint{0cm}{1.359cm}}{\pgfqpoint{0.014cm}{1.324cm}}{\pgfqpoint{0.04cm}{1.299cm}}
\pgfpathcurveto{\pgfqpoint{0.066cm}{1.273cm}}{\pgfqpoint{0.1cm}{1.258cm}}{\pgfqpoint{0.137cm}{1.258cm}}
\pgfpathcurveto{\pgfqpoint{0.173cm}{1.258cm}}{\pgfqpoint{0.207cm}{1.273cm}}{\pgfqpoint{0.233cm}{1.299cm}}
\pgfpathcurveto{\pgfqpoint{0.259cm}{1.324cm}}{\pgfqpoint{0.273cm}{1.359cm}}{\pgfqpoint{0.273cm}{1.395cm}}
\pgfusepath{fill}
\begin{pgfscope}
\pgfsetdash{}{0cm}
\pgfsetlinewidth{0.818mm}
\pgfsetmiterlimit{7.0}
\pgfpathmoveto{\pgfqpoint{0.682cm}{0.671cm}}
\pgfpathlineto{\pgfqpoint{0.679cm}{1.418cm}}
\pgfusepath{stroke}
\end{pgfscope}
\pgfpathmoveto{\pgfqpoint{0.815cm}{1.399cm}}
\pgfpathcurveto{\pgfqpoint{0.815cm}{1.435cm}}{\pgfqpoint{0.801cm}{1.47cm}}{\pgfqpoint{0.775cm}{1.496cm}}
\pgfpathcurveto{\pgfqpoint{0.75cm}{1.521cm}}{\pgfqpoint{0.715cm}{1.536cm}}{\pgfqpoint{0.679cm}{1.536cm}}
\pgfpathcurveto{\pgfqpoint{0.643cm}{1.536cm}}{\pgfqpoint{0.608cm}{1.521cm}}{\pgfqpoint{0.582cm}{1.496cm}}
\pgfpathcurveto{\pgfqpoint{0.557cm}{1.47cm}}{\pgfqpoint{0.542cm}{1.435cm}}{\pgfqpoint{0.542cm}{1.399cm}}
\pgfpathcurveto{\pgfqpoint{0.542cm}{1.363cm}}{\pgfqpoint{0.557cm}{1.328cm}}{\pgfqpoint{0.582cm}{1.302cm}}
\pgfpathcurveto{\pgfqpoint{0.608cm}{1.276cm}}{\pgfqpoint{0.643cm}{1.262cm}}{\pgfqpoint{0.679cm}{1.262cm}}
\pgfpathcurveto{\pgfqpoint{0.715cm}{1.262cm}}{\pgfqpoint{0.75cm}{1.276cm}}{\pgfqpoint{0.775cm}{1.302cm}}
\pgfpathcurveto{\pgfqpoint{0.801cm}{1.328cm}}{\pgfqpoint{0.815cm}{1.363cm}}{\pgfqpoint{0.815cm}{1.399cm}}
\pgfusepath{fill}
\pgfpathmoveto{\pgfqpoint{1.345cm}{1.371cm}}
\pgfpathcurveto{\pgfqpoint{1.345cm}{1.408cm}}{\pgfqpoint{1.331cm}{1.442cm}}{\pgfqpoint{1.305cm}{1.468cm}}
\pgfpathcurveto{\pgfqpoint{1.28cm}{1.494cm}}{\pgfqpoint{1.245cm}{1.508cm}}{\pgfqpoint{1.209cm}{1.508cm}}
\pgfpathcurveto{\pgfqpoint{1.172cm}{1.508cm}}{\pgfqpoint{1.138cm}{1.494cm}}{\pgfqpoint{1.112cm}{1.468cm}}
\pgfpathcurveto{\pgfqpoint{1.087cm}{1.442cm}}{\pgfqpoint{1.072cm}{1.408cm}}{\pgfqpoint{1.072cm}{1.371cm}}
\pgfpathcurveto{\pgfqpoint{1.072cm}{1.335cm}}{\pgfqpoint{1.087cm}{1.3cm}}{\pgfqpoint{1.112cm}{1.274cm}}
\pgfpathcurveto{\pgfqpoint{1.138cm}{1.249cm}}{\pgfqpoint{1.172cm}{1.234cm}}{\pgfqpoint{1.209cm}{1.234cm}}
\pgfpathcurveto{\pgfqpoint{1.245cm}{1.234cm}}{\pgfqpoint{1.28cm}{1.249cm}}{\pgfqpoint{1.305cm}{1.274cm}}
\pgfpathcurveto{\pgfqpoint{1.331cm}{1.3cm}}{\pgfqpoint{1.345cm}{1.335cm}}{\pgfqpoint{1.345cm}{1.371cm}}
\pgfusepath{fill}
\begin{pgfscope}
\pgfsetdash{}{0cm}
\pgfsetlinewidth{0.818mm}
\pgfsetroundcap
\pgfsetmiterlimit{4.0}
\pgfpathmoveto{\pgfqpoint{0.682cm}{0.671cm}}
\pgfpathlineto{\pgfqpoint{0.682cm}{0.042cm}}
\pgfusepath{stroke}
\end{pgfscope}
\end{pgfscope}
\end{pgfscope}
\end{pgfscope}
\end{tikzpicture}}}, X^{\!\resizebox{0.6em}{!}{
\begin{tikzpicture}
\pgfpathmoveto{\pgfqpoint{0cm}{0cm}}
\pgfpathlineto{\pgfqpoint{1.376cm}{0cm}}
\pgfpathlineto{\pgfqpoint{1.376cm}{1.588cm}}
\pgfpathlineto{\pgfqpoint{0cm}{1.588cm}}
\pgfpathclose
\pgfusepath{clip}
\begin{pgfscope}
\begin{pgfscope}
\pgfpathmoveto{\pgfqpoint{0cm}{0cm}}
\pgfpathlineto{\pgfqpoint{1.376cm}{0cm}}
\pgfpathlineto{\pgfqpoint{1.376cm}{1.588cm}}
\pgfpathlineto{\pgfqpoint{0cm}{1.588cm}}
\pgfpathclose
\pgfusepath{clip}
\begin{pgfscope}
\begin{pgfscope}
\definecolor{eps2pgf_color}{gray}{0.976471}\pgfsetstrokecolor{eps2pgf_color}\pgfsetfillcolor{eps2pgf_color}
\pgfpathmoveto{\pgfqpoint{0cm}{0cm}}
\pgfpathlineto{\pgfqpoint{1.376cm}{0cm}}
\pgfpathlineto{\pgfqpoint{1.376cm}{1.588cm}}
\pgfpathlineto{\pgfqpoint{0cm}{1.588cm}}
\pgfpathclose
\pgfusepath{fill}
\end{pgfscope}
\begin{pgfscope}
\pgfsetdash{}{0cm}
\pgfsetlinewidth{0.818mm}
\pgfsetroundcap
\pgfsetroundjoin
\pgfsetmiterlimit{7.0}
\definecolor{eps2pgf_color}{gray}{0}\pgfsetstrokecolor{eps2pgf_color}\pgfsetfillcolor{eps2pgf_color}
\pgfpathmoveto{\pgfqpoint{0.117cm}{1.476cm}}
\pgfpathlineto{\pgfqpoint{0.682cm}{0.726cm}}
\pgfpathlineto{\pgfqpoint{1.246cm}{1.476cm}}
\pgfusepath{stroke}
\end{pgfscope}
\definecolor{eps2pgf_color}{gray}{0}\pgfsetstrokecolor{eps2pgf_color}\pgfsetfillcolor{eps2pgf_color}
\pgfpathmoveto{\pgfqpoint{0.273cm}{1.451cm}}
\pgfpathcurveto{\pgfqpoint{0.273cm}{1.487cm}}{\pgfqpoint{0.259cm}{1.522cm}}{\pgfqpoint{0.233cm}{1.547cm}}
\pgfpathcurveto{\pgfqpoint{0.207cm}{1.573cm}}{\pgfqpoint{0.173cm}{1.588cm}}{\pgfqpoint{0.137cm}{1.588cm}}
\pgfpathcurveto{\pgfqpoint{0.1cm}{1.588cm}}{\pgfqpoint{0.066cm}{1.573cm}}{\pgfqpoint{0.04cm}{1.547cm}}
\pgfpathcurveto{\pgfqpoint{0.014cm}{1.522cm}}{\pgfqpoint{0cm}{1.487cm}}{\pgfqpoint{0cm}{1.451cm}}
\pgfpathcurveto{\pgfqpoint{0cm}{1.414cm}}{\pgfqpoint{0.014cm}{1.379cm}}{\pgfqpoint{0.04cm}{1.354cm}}
\pgfpathcurveto{\pgfqpoint{0.066cm}{1.328cm}}{\pgfqpoint{0.1cm}{1.314cm}}{\pgfqpoint{0.137cm}{1.314cm}}
\pgfpathcurveto{\pgfqpoint{0.173cm}{1.314cm}}{\pgfqpoint{0.207cm}{1.328cm}}{\pgfqpoint{0.233cm}{1.354cm}}
\pgfpathcurveto{\pgfqpoint{0.259cm}{1.379cm}}{\pgfqpoint{0.273cm}{1.414cm}}{\pgfqpoint{0.273cm}{1.451cm}}
\pgfusepath{fill}
\pgfpathmoveto{\pgfqpoint{1.345cm}{1.426cm}}
\pgfpathcurveto{\pgfqpoint{1.345cm}{1.463cm}}{\pgfqpoint{1.331cm}{1.497cm}}{\pgfqpoint{1.305cm}{1.523cm}}
\pgfpathcurveto{\pgfqpoint{1.28cm}{1.549cm}}{\pgfqpoint{1.245cm}{1.563cm}}{\pgfqpoint{1.209cm}{1.563cm}}
\pgfpathcurveto{\pgfqpoint{1.172cm}{1.563cm}}{\pgfqpoint{1.138cm}{1.549cm}}{\pgfqpoint{1.112cm}{1.523cm}}
\pgfpathcurveto{\pgfqpoint{1.087cm}{1.497cm}}{\pgfqpoint{1.072cm}{1.463cm}}{\pgfqpoint{1.072cm}{1.426cm}}
\pgfpathcurveto{\pgfqpoint{1.072cm}{1.39cm}}{\pgfqpoint{1.087cm}{1.355cm}}{\pgfqpoint{1.112cm}{1.329cm}}
\pgfpathcurveto{\pgfqpoint{1.138cm}{1.304cm}}{\pgfqpoint{1.172cm}{1.289cm}}{\pgfqpoint{1.209cm}{1.289cm}}
\pgfpathcurveto{\pgfqpoint{1.245cm}{1.289cm}}{\pgfqpoint{1.28cm}{1.304cm}}{\pgfqpoint{1.305cm}{1.329cm}}
\pgfpathcurveto{\pgfqpoint{1.331cm}{1.355cm}}{\pgfqpoint{1.345cm}{1.39cm}}{\pgfqpoint{1.345cm}{1.426cm}}
\pgfusepath{fill}
\begin{pgfscope}
\pgfsetdash{}{0cm}
\pgfsetlinewidth{0.818mm}
\pgfsetroundcap
\pgfsetmiterlimit{4.0}
\pgfpathmoveto{\pgfqpoint{0.682cm}{0.726cm}}
\pgfpathlineto{\pgfqpoint{0.682cm}{0.097cm}}
\pgfusepath{stroke}
\end{pgfscope}
\end{pgfscope}
\end{pgfscope}
\end{pgfscope}
\end{tikzpicture}}}, X^{\!\resizebox{!}{.8em}{
\begin{tikzpicture}
\pgfpathmoveto{\pgfqpoint{0cm}{-0.035cm}}
\pgfpathlineto{\pgfqpoint{1.976cm}{-0.035cm}}
\pgfpathlineto{\pgfqpoint{1.976cm}{1.94cm}}
\pgfpathlineto{\pgfqpoint{0cm}{1.94cm}}
\pgfpathclose
\pgfusepath{clip}
\begin{pgfscope}
\begin{pgfscope}
\pgfpathmoveto{\pgfqpoint{0cm}{-0.035cm}}
\pgfpathlineto{\pgfqpoint{1.976cm}{-0.035cm}}
\pgfpathlineto{\pgfqpoint{1.976cm}{1.94cm}}
\pgfpathlineto{\pgfqpoint{0cm}{1.94cm}}
\pgfpathclose
\pgfusepath{clip}
\begin{pgfscope}
\begin{pgfscope}
\pgfsetdash{}{0cm}
\pgfsetlinewidth{0.818mm}
\pgfsetroundcap
\pgfsetroundjoin
\pgfsetmiterlimit{7.0}
\definecolor{eps2pgf_color}{gray}{0}\pgfsetstrokecolor{eps2pgf_color}\pgfsetfillcolor{eps2pgf_color}
\pgfpathmoveto{\pgfqpoint{0.117cm}{1.815cm}}
\pgfpathlineto{\pgfqpoint{0.682cm}{1.065cm}}
\pgfpathlineto{\pgfqpoint{1.246cm}{1.815cm}}
\pgfusepath{stroke}
\end{pgfscope}
\definecolor{eps2pgf_color}{gray}{0}\pgfsetstrokecolor{eps2pgf_color}\pgfsetfillcolor{eps2pgf_color}
\pgfpathmoveto{\pgfqpoint{0.273cm}{1.789cm}}
\pgfpathcurveto{\pgfqpoint{0.273cm}{1.825cm}}{\pgfqpoint{0.259cm}{1.86cm}}{\pgfqpoint{0.233cm}{1.886cm}}
\pgfpathcurveto{\pgfqpoint{0.207cm}{1.912cm}}{\pgfqpoint{0.173cm}{1.926cm}}{\pgfqpoint{0.137cm}{1.926cm}}
\pgfpathcurveto{\pgfqpoint{0.1cm}{1.926cm}}{\pgfqpoint{0.066cm}{1.912cm}}{\pgfqpoint{0.04cm}{1.886cm}}
\pgfpathcurveto{\pgfqpoint{0.014cm}{1.86cm}}{\pgfqpoint{0cm}{1.825cm}}{\pgfqpoint{0cm}{1.789cm}}
\pgfpathcurveto{\pgfqpoint{0cm}{1.753cm}}{\pgfqpoint{0.014cm}{1.718cm}}{\pgfqpoint{0.04cm}{1.692cm}}
\pgfpathcurveto{\pgfqpoint{0.066cm}{1.667cm}}{\pgfqpoint{0.1cm}{1.652cm}}{\pgfqpoint{0.137cm}{1.652cm}}
\pgfpathcurveto{\pgfqpoint{0.173cm}{1.652cm}}{\pgfqpoint{0.207cm}{1.667cm}}{\pgfqpoint{0.233cm}{1.692cm}}
\pgfpathcurveto{\pgfqpoint{0.259cm}{1.718cm}}{\pgfqpoint{0.273cm}{1.753cm}}{\pgfqpoint{0.273cm}{1.789cm}}
\pgfusepath{fill}
\begin{pgfscope}
\pgfsetdash{}{0cm}
\pgfsetlinewidth{0.818mm}
\pgfsetmiterlimit{7.0}
\pgfpathmoveto{\pgfqpoint{0.682cm}{1.065cm}}
\pgfpathlineto{\pgfqpoint{0.679cm}{1.812cm}}
\pgfusepath{stroke}
\end{pgfscope}
\pgfpathmoveto{\pgfqpoint{0.815cm}{1.793cm}}
\pgfpathcurveto{\pgfqpoint{0.815cm}{1.829cm}}{\pgfqpoint{0.801cm}{1.864cm}}{\pgfqpoint{0.775cm}{1.89cm}}
\pgfpathcurveto{\pgfqpoint{0.75cm}{1.915cm}}{\pgfqpoint{0.715cm}{1.93cm}}{\pgfqpoint{0.679cm}{1.93cm}}
\pgfpathcurveto{\pgfqpoint{0.643cm}{1.93cm}}{\pgfqpoint{0.608cm}{1.915cm}}{\pgfqpoint{0.582cm}{1.89cm}}
\pgfpathcurveto{\pgfqpoint{0.557cm}{1.864cm}}{\pgfqpoint{0.542cm}{1.829cm}}{\pgfqpoint{0.542cm}{1.793cm}}
\pgfpathcurveto{\pgfqpoint{0.542cm}{1.756cm}}{\pgfqpoint{0.557cm}{1.722cm}}{\pgfqpoint{0.582cm}{1.696cm}}
\pgfpathcurveto{\pgfqpoint{0.608cm}{1.67cm}}{\pgfqpoint{0.643cm}{1.656cm}}{\pgfqpoint{0.679cm}{1.656cm}}
\pgfpathcurveto{\pgfqpoint{0.715cm}{1.656cm}}{\pgfqpoint{0.75cm}{1.67cm}}{\pgfqpoint{0.775cm}{1.696cm}}
\pgfpathcurveto{\pgfqpoint{0.801cm}{1.722cm}}{\pgfqpoint{0.815cm}{1.756cm}}{\pgfqpoint{0.815cm}{1.793cm}}
\pgfusepath{fill}
\pgfpathmoveto{\pgfqpoint{1.345cm}{1.765cm}}
\pgfpathcurveto{\pgfqpoint{1.345cm}{1.801cm}}{\pgfqpoint{1.331cm}{1.836cm}}{\pgfqpoint{1.305cm}{1.862cm}}
\pgfpathcurveto{\pgfqpoint{1.28cm}{1.887cm}}{\pgfqpoint{1.245cm}{1.902cm}}{\pgfqpoint{1.209cm}{1.902cm}}
\pgfpathcurveto{\pgfqpoint{1.172cm}{1.902cm}}{\pgfqpoint{1.138cm}{1.887cm}}{\pgfqpoint{1.112cm}{1.862cm}}
\pgfpathcurveto{\pgfqpoint{1.087cm}{1.836cm}}{\pgfqpoint{1.072cm}{1.801cm}}{\pgfqpoint{1.072cm}{1.765cm}}
\pgfpathcurveto{\pgfqpoint{1.072cm}{1.728cm}}{\pgfqpoint{1.087cm}{1.694cm}}{\pgfqpoint{1.112cm}{1.668cm}}
\pgfpathcurveto{\pgfqpoint{1.138cm}{1.642cm}}{\pgfqpoint{1.172cm}{1.628cm}}{\pgfqpoint{1.209cm}{1.628cm}}
\pgfpathcurveto{\pgfqpoint{1.245cm}{1.628cm}}{\pgfqpoint{1.28cm}{1.642cm}}{\pgfqpoint{1.305cm}{1.668cm}}
\pgfpathcurveto{\pgfqpoint{1.331cm}{1.694cm}}{\pgfqpoint{1.345cm}{1.728cm}}{\pgfqpoint{1.345cm}{1.765cm}}
\pgfusepath{fill}
\begin{pgfscope}
\pgfsetdash{}{0cm}
\pgfsetlinewidth{0.818mm}
\pgfsetroundcap
\pgfsetroundjoin
\pgfsetmiterlimit{7.0}
\pgfpathmoveto{\pgfqpoint{0.682cm}{1.065cm}}
\pgfpathlineto{\pgfqpoint{1.246cm}{0.315cm}}
\pgfpathlineto{\pgfqpoint{1.811cm}{1.065cm}}
\pgfusepath{stroke}
\end{pgfscope}
\pgfpathmoveto{\pgfqpoint{1.948cm}{1.065cm}}
\pgfpathcurveto{\pgfqpoint{1.948cm}{1.101cm}}{\pgfqpoint{1.933cm}{1.136cm}}{\pgfqpoint{1.907cm}{1.162cm}}
\pgfpathcurveto{\pgfqpoint{1.882cm}{1.187cm}}{\pgfqpoint{1.847cm}{1.202cm}}{\pgfqpoint{1.811cm}{1.202cm}}
\pgfpathcurveto{\pgfqpoint{1.775cm}{1.202cm}}{\pgfqpoint{1.74cm}{1.187cm}}{\pgfqpoint{1.714cm}{1.162cm}}
\pgfpathcurveto{\pgfqpoint{1.689cm}{1.136cm}}{\pgfqpoint{1.674cm}{1.101cm}}{\pgfqpoint{1.674cm}{1.065cm}}
\pgfpathcurveto{\pgfqpoint{1.674cm}{1.029cm}}{\pgfqpoint{1.689cm}{0.994cm}}{\pgfqpoint{1.714cm}{0.968cm}}
\pgfpathcurveto{\pgfqpoint{1.74cm}{0.942cm}}{\pgfqpoint{1.775cm}{0.928cm}}{\pgfqpoint{1.811cm}{0.928cm}}
\pgfpathcurveto{\pgfqpoint{1.847cm}{0.928cm}}{\pgfqpoint{1.882cm}{0.942cm}}{\pgfqpoint{1.907cm}{0.968cm}}
\pgfpathcurveto{\pgfqpoint{1.933cm}{0.994cm}}{\pgfqpoint{1.948cm}{1.029cm}}{\pgfqpoint{1.948cm}{1.065cm}}
\pgfusepath{fill}
\begin{pgfscope}
\pgfsetdash{}{0cm}
\pgfsetlinewidth{0.818mm}
\pgfsetmiterlimit{4.0}
\pgfpathmoveto{\pgfqpoint{1.383cm}{0.178cm}}
\pgfpathcurveto{\pgfqpoint{1.383cm}{0.214cm}}{\pgfqpoint{1.369cm}{0.249cm}}{\pgfqpoint{1.343cm}{0.275cm}}
\pgfpathcurveto{\pgfqpoint{1.317cm}{0.3cm}}{\pgfqpoint{1.283cm}{0.315cm}}{\pgfqpoint{1.246cm}{0.315cm}}
\pgfpathcurveto{\pgfqpoint{1.21cm}{0.315cm}}{\pgfqpoint{1.175cm}{0.3cm}}{\pgfqpoint{1.15cm}{0.275cm}}
\pgfpathcurveto{\pgfqpoint{1.124cm}{0.249cm}}{\pgfqpoint{1.11cm}{0.214cm}}{\pgfqpoint{1.11cm}{0.178cm}}
\pgfpathcurveto{\pgfqpoint{1.11cm}{0.141cm}}{\pgfqpoint{1.124cm}{0.107cm}}{\pgfqpoint{1.15cm}{0.081cm}}
\pgfpathcurveto{\pgfqpoint{1.175cm}{0.055cm}}{\pgfqpoint{1.21cm}{0.041cm}}{\pgfqpoint{1.246cm}{0.041cm}}
\pgfpathcurveto{\pgfqpoint{1.283cm}{0.041cm}}{\pgfqpoint{1.317cm}{0.055cm}}{\pgfqpoint{1.343cm}{0.081cm}}
\pgfpathcurveto{\pgfqpoint{1.369cm}{0.107cm}}{\pgfqpoint{1.383cm}{0.141cm}}{\pgfqpoint{1.383cm}{0.178cm}}
\pgfusepath{stroke}
\end{pgfscope}
\end{pgfscope}
\end{pgfscope}
\end{pgfscope}
\end{tikzpicture}}}, X^{\!\resizebox{!}{.8em}{
\begin{tikzpicture}
\pgfpathmoveto{\pgfqpoint{0cm}{-0.035cm}}
\pgfpathlineto{\pgfqpoint{1.976cm}{-0.035cm}}
\pgfpathlineto{\pgfqpoint{1.976cm}{1.94cm}}
\pgfpathlineto{\pgfqpoint{0cm}{1.94cm}}
\pgfpathclose
\pgfusepath{clip}
\begin{pgfscope}
\begin{pgfscope}
\pgfpathmoveto{\pgfqpoint{0cm}{-0.035cm}}
\pgfpathlineto{\pgfqpoint{1.976cm}{-0.035cm}}
\pgfpathlineto{\pgfqpoint{1.976cm}{1.94cm}}
\pgfpathlineto{\pgfqpoint{0cm}{1.94cm}}
\pgfpathclose
\pgfusepath{clip}
\begin{pgfscope}
\begin{pgfscope}
\pgfsetdash{}{0cm}
\pgfsetlinewidth{0.818mm}
\pgfsetroundcap
\pgfsetroundjoin
\pgfsetmiterlimit{7.0}
\definecolor{eps2pgf_color}{gray}{0}\pgfsetstrokecolor{eps2pgf_color}\pgfsetfillcolor{eps2pgf_color}
\pgfpathmoveto{\pgfqpoint{0.117cm}{1.815cm}}
\pgfpathlineto{\pgfqpoint{0.682cm}{1.065cm}}
\pgfpathlineto{\pgfqpoint{1.246cm}{1.815cm}}
\pgfusepath{stroke}
\end{pgfscope}
\definecolor{eps2pgf_color}{gray}{0}\pgfsetstrokecolor{eps2pgf_color}\pgfsetfillcolor{eps2pgf_color}
\pgfpathmoveto{\pgfqpoint{0.273cm}{1.789cm}}
\pgfpathcurveto{\pgfqpoint{0.273cm}{1.825cm}}{\pgfqpoint{0.259cm}{1.86cm}}{\pgfqpoint{0.233cm}{1.886cm}}
\pgfpathcurveto{\pgfqpoint{0.207cm}{1.912cm}}{\pgfqpoint{0.173cm}{1.926cm}}{\pgfqpoint{0.137cm}{1.926cm}}
\pgfpathcurveto{\pgfqpoint{0.1cm}{1.926cm}}{\pgfqpoint{0.066cm}{1.912cm}}{\pgfqpoint{0.04cm}{1.886cm}}
\pgfpathcurveto{\pgfqpoint{0.014cm}{1.86cm}}{\pgfqpoint{0cm}{1.825cm}}{\pgfqpoint{0cm}{1.789cm}}
\pgfpathcurveto{\pgfqpoint{0cm}{1.753cm}}{\pgfqpoint{0.014cm}{1.718cm}}{\pgfqpoint{0.04cm}{1.692cm}}
\pgfpathcurveto{\pgfqpoint{0.066cm}{1.667cm}}{\pgfqpoint{0.1cm}{1.652cm}}{\pgfqpoint{0.137cm}{1.652cm}}
\pgfpathcurveto{\pgfqpoint{0.173cm}{1.652cm}}{\pgfqpoint{0.207cm}{1.667cm}}{\pgfqpoint{0.233cm}{1.692cm}}
\pgfpathcurveto{\pgfqpoint{0.259cm}{1.718cm}}{\pgfqpoint{0.273cm}{1.753cm}}{\pgfqpoint{0.273cm}{1.789cm}}
\pgfusepath{fill}
\pgfpathmoveto{\pgfqpoint{1.345cm}{1.765cm}}
\pgfpathcurveto{\pgfqpoint{1.345cm}{1.801cm}}{\pgfqpoint{1.331cm}{1.836cm}}{\pgfqpoint{1.305cm}{1.862cm}}
\pgfpathcurveto{\pgfqpoint{1.28cm}{1.887cm}}{\pgfqpoint{1.245cm}{1.902cm}}{\pgfqpoint{1.209cm}{1.902cm}}
\pgfpathcurveto{\pgfqpoint{1.172cm}{1.902cm}}{\pgfqpoint{1.138cm}{1.887cm}}{\pgfqpoint{1.112cm}{1.862cm}}
\pgfpathcurveto{\pgfqpoint{1.087cm}{1.836cm}}{\pgfqpoint{1.072cm}{1.801cm}}{\pgfqpoint{1.072cm}{1.765cm}}
\pgfpathcurveto{\pgfqpoint{1.072cm}{1.728cm}}{\pgfqpoint{1.087cm}{1.694cm}}{\pgfqpoint{1.112cm}{1.668cm}}
\pgfpathcurveto{\pgfqpoint{1.138cm}{1.642cm}}{\pgfqpoint{1.172cm}{1.628cm}}{\pgfqpoint{1.209cm}{1.628cm}}
\pgfpathcurveto{\pgfqpoint{1.245cm}{1.628cm}}{\pgfqpoint{1.28cm}{1.642cm}}{\pgfqpoint{1.305cm}{1.668cm}}
\pgfpathcurveto{\pgfqpoint{1.331cm}{1.694cm}}{\pgfqpoint{1.345cm}{1.728cm}}{\pgfqpoint{1.345cm}{1.765cm}}
\pgfusepath{fill}
\begin{pgfscope}
\pgfsetdash{}{0cm}
\pgfsetlinewidth{0.818mm}
\pgfsetroundcap
\pgfsetroundjoin
\pgfsetmiterlimit{7.0}
\pgfpathmoveto{\pgfqpoint{0.682cm}{1.065cm}}
\pgfpathlineto{\pgfqpoint{1.246cm}{0.315cm}}
\pgfpathlineto{\pgfqpoint{1.811cm}{1.065cm}}
\pgfusepath{stroke}
\end{pgfscope}
\pgfpathmoveto{\pgfqpoint{1.948cm}{1.065cm}}
\pgfpathcurveto{\pgfqpoint{1.948cm}{1.101cm}}{\pgfqpoint{1.933cm}{1.136cm}}{\pgfqpoint{1.907cm}{1.162cm}}
\pgfpathcurveto{\pgfqpoint{1.882cm}{1.187cm}}{\pgfqpoint{1.847cm}{1.202cm}}{\pgfqpoint{1.811cm}{1.202cm}}
\pgfpathcurveto{\pgfqpoint{1.775cm}{1.202cm}}{\pgfqpoint{1.74cm}{1.187cm}}{\pgfqpoint{1.714cm}{1.162cm}}
\pgfpathcurveto{\pgfqpoint{1.689cm}{1.136cm}}{\pgfqpoint{1.674cm}{1.101cm}}{\pgfqpoint{1.674cm}{1.065cm}}
\pgfpathcurveto{\pgfqpoint{1.674cm}{1.029cm}}{\pgfqpoint{1.689cm}{0.994cm}}{\pgfqpoint{1.714cm}{0.968cm}}
\pgfpathcurveto{\pgfqpoint{1.74cm}{0.942cm}}{\pgfqpoint{1.775cm}{0.928cm}}{\pgfqpoint{1.811cm}{0.928cm}}
\pgfpathcurveto{\pgfqpoint{1.847cm}{0.928cm}}{\pgfqpoint{1.882cm}{0.942cm}}{\pgfqpoint{1.907cm}{0.968cm}}
\pgfpathcurveto{\pgfqpoint{1.933cm}{0.994cm}}{\pgfqpoint{1.948cm}{1.029cm}}{\pgfqpoint{1.948cm}{1.065cm}}
\pgfusepath{fill}
\begin{pgfscope}
\pgfsetdash{}{0cm}
\pgfsetlinewidth{0.818mm}
\pgfsetmiterlimit{7.0}
\pgfpathmoveto{\pgfqpoint{1.246cm}{0.315cm}}
\pgfpathlineto{\pgfqpoint{1.244cm}{1.061cm}}
\pgfusepath{stroke}
\end{pgfscope}
\pgfpathmoveto{\pgfqpoint{1.38cm}{1.065cm}}
\pgfpathcurveto{\pgfqpoint{1.38cm}{1.101cm}}{\pgfqpoint{1.366cm}{1.136cm}}{\pgfqpoint{1.34cm}{1.162cm}}
\pgfpathcurveto{\pgfqpoint{1.315cm}{1.187cm}}{\pgfqpoint{1.28cm}{1.202cm}}{\pgfqpoint{1.244cm}{1.202cm}}
\pgfpathcurveto{\pgfqpoint{1.207cm}{1.202cm}}{\pgfqpoint{1.173cm}{1.187cm}}{\pgfqpoint{1.147cm}{1.162cm}}
\pgfpathcurveto{\pgfqpoint{1.121cm}{1.136cm}}{\pgfqpoint{1.107cm}{1.101cm}}{\pgfqpoint{1.107cm}{1.065cm}}
\pgfpathcurveto{\pgfqpoint{1.107cm}{1.029cm}}{\pgfqpoint{1.121cm}{0.994cm}}{\pgfqpoint{1.147cm}{0.968cm}}
\pgfpathcurveto{\pgfqpoint{1.173cm}{0.942cm}}{\pgfqpoint{1.207cm}{0.928cm}}{\pgfqpoint{1.244cm}{0.928cm}}
\pgfpathcurveto{\pgfqpoint{1.28cm}{0.928cm}}{\pgfqpoint{1.315cm}{0.942cm}}{\pgfqpoint{1.34cm}{0.968cm}}
\pgfpathcurveto{\pgfqpoint{1.366cm}{0.994cm}}{\pgfqpoint{1.38cm}{1.029cm}}{\pgfqpoint{1.38cm}{1.065cm}}
\pgfusepath{fill}
\begin{pgfscope}
\pgfsetdash{}{0cm}
\pgfsetlinewidth{0.818mm}
\pgfsetmiterlimit{4.0}
\pgfpathmoveto{\pgfqpoint{1.383cm}{0.178cm}}
\pgfpathcurveto{\pgfqpoint{1.383cm}{0.214cm}}{\pgfqpoint{1.369cm}{0.249cm}}{\pgfqpoint{1.343cm}{0.275cm}}
\pgfpathcurveto{\pgfqpoint{1.317cm}{0.3cm}}{\pgfqpoint{1.283cm}{0.315cm}}{\pgfqpoint{1.246cm}{0.315cm}}
\pgfpathcurveto{\pgfqpoint{1.21cm}{0.315cm}}{\pgfqpoint{1.175cm}{0.3cm}}{\pgfqpoint{1.15cm}{0.275cm}}
\pgfpathcurveto{\pgfqpoint{1.124cm}{0.249cm}}{\pgfqpoint{1.11cm}{0.214cm}}{\pgfqpoint{1.11cm}{0.178cm}}
\pgfpathcurveto{\pgfqpoint{1.11cm}{0.141cm}}{\pgfqpoint{1.124cm}{0.107cm}}{\pgfqpoint{1.15cm}{0.081cm}}
\pgfpathcurveto{\pgfqpoint{1.175cm}{0.055cm}}{\pgfqpoint{1.21cm}{0.041cm}}{\pgfqpoint{1.246cm}{0.041cm}}
\pgfpathcurveto{\pgfqpoint{1.283cm}{0.041cm}}{\pgfqpoint{1.317cm}{0.055cm}}{\pgfqpoint{1.343cm}{0.081cm}}
\pgfpathcurveto{\pgfqpoint{1.369cm}{0.107cm}}{\pgfqpoint{1.383cm}{0.141cm}}{\pgfqpoint{1.383cm}{0.178cm}}
\pgfusepath{stroke}
\end{pgfscope}
\end{pgfscope}
\end{pgfscope}
\end{pgfscope}
\end{tikzpicture}}}, X^{\!\resizebox{!}{.8em}{
\begin{tikzpicture}
\pgfpathmoveto{\pgfqpoint{0cm}{-0.035cm}}
\pgfpathlineto{\pgfqpoint{1.976cm}{-0.035cm}}
\pgfpathlineto{\pgfqpoint{1.976cm}{1.94cm}}
\pgfpathlineto{\pgfqpoint{0cm}{1.94cm}}
\pgfpathclose
\pgfusepath{clip}
\begin{pgfscope}
\begin{pgfscope}
\pgfpathmoveto{\pgfqpoint{0cm}{-0.035cm}}
\pgfpathlineto{\pgfqpoint{1.976cm}{-0.035cm}}
\pgfpathlineto{\pgfqpoint{1.976cm}{1.94cm}}
\pgfpathlineto{\pgfqpoint{0cm}{1.94cm}}
\pgfpathclose
\pgfusepath{clip}
\begin{pgfscope}
\begin{pgfscope}
\pgfsetdash{}{0cm}
\pgfsetlinewidth{0.818mm}
\pgfsetroundcap
\pgfsetroundjoin
\pgfsetmiterlimit{7.0}
\definecolor{eps2pgf_color}{gray}{0}\pgfsetstrokecolor{eps2pgf_color}\pgfsetfillcolor{eps2pgf_color}
\pgfpathmoveto{\pgfqpoint{0.117cm}{1.815cm}}
\pgfpathlineto{\pgfqpoint{0.682cm}{1.065cm}}
\pgfpathlineto{\pgfqpoint{1.246cm}{1.815cm}}
\pgfusepath{stroke}
\end{pgfscope}
\definecolor{eps2pgf_color}{gray}{0}\pgfsetstrokecolor{eps2pgf_color}\pgfsetfillcolor{eps2pgf_color}
\pgfpathmoveto{\pgfqpoint{0.273cm}{1.789cm}}
\pgfpathcurveto{\pgfqpoint{0.273cm}{1.825cm}}{\pgfqpoint{0.259cm}{1.86cm}}{\pgfqpoint{0.233cm}{1.886cm}}
\pgfpathcurveto{\pgfqpoint{0.207cm}{1.912cm}}{\pgfqpoint{0.173cm}{1.926cm}}{\pgfqpoint{0.137cm}{1.926cm}}
\pgfpathcurveto{\pgfqpoint{0.1cm}{1.926cm}}{\pgfqpoint{0.066cm}{1.912cm}}{\pgfqpoint{0.04cm}{1.886cm}}
\pgfpathcurveto{\pgfqpoint{0.014cm}{1.86cm}}{\pgfqpoint{0cm}{1.825cm}}{\pgfqpoint{0cm}{1.789cm}}
\pgfpathcurveto{\pgfqpoint{0cm}{1.753cm}}{\pgfqpoint{0.014cm}{1.718cm}}{\pgfqpoint{0.04cm}{1.692cm}}
\pgfpathcurveto{\pgfqpoint{0.066cm}{1.667cm}}{\pgfqpoint{0.1cm}{1.652cm}}{\pgfqpoint{0.137cm}{1.652cm}}
\pgfpathcurveto{\pgfqpoint{0.173cm}{1.652cm}}{\pgfqpoint{0.207cm}{1.667cm}}{\pgfqpoint{0.233cm}{1.692cm}}
\pgfpathcurveto{\pgfqpoint{0.259cm}{1.718cm}}{\pgfqpoint{0.273cm}{1.753cm}}{\pgfqpoint{0.273cm}{1.789cm}}
\pgfusepath{fill}
\begin{pgfscope}
\pgfsetdash{}{0cm}
\pgfsetlinewidth{0.818mm}
\pgfsetmiterlimit{7.0}
\pgfpathmoveto{\pgfqpoint{0.682cm}{1.065cm}}
\pgfpathlineto{\pgfqpoint{0.679cm}{1.812cm}}
\pgfusepath{stroke}
\end{pgfscope}
\pgfpathmoveto{\pgfqpoint{0.815cm}{1.793cm}}
\pgfpathcurveto{\pgfqpoint{0.815cm}{1.829cm}}{\pgfqpoint{0.801cm}{1.864cm}}{\pgfqpoint{0.775cm}{1.89cm}}
\pgfpathcurveto{\pgfqpoint{0.75cm}{1.915cm}}{\pgfqpoint{0.715cm}{1.93cm}}{\pgfqpoint{0.679cm}{1.93cm}}
\pgfpathcurveto{\pgfqpoint{0.643cm}{1.93cm}}{\pgfqpoint{0.608cm}{1.915cm}}{\pgfqpoint{0.582cm}{1.89cm}}
\pgfpathcurveto{\pgfqpoint{0.557cm}{1.864cm}}{\pgfqpoint{0.542cm}{1.829cm}}{\pgfqpoint{0.542cm}{1.793cm}}
\pgfpathcurveto{\pgfqpoint{0.542cm}{1.756cm}}{\pgfqpoint{0.557cm}{1.722cm}}{\pgfqpoint{0.582cm}{1.696cm}}
\pgfpathcurveto{\pgfqpoint{0.608cm}{1.67cm}}{\pgfqpoint{0.643cm}{1.656cm}}{\pgfqpoint{0.679cm}{1.656cm}}
\pgfpathcurveto{\pgfqpoint{0.715cm}{1.656cm}}{\pgfqpoint{0.75cm}{1.67cm}}{\pgfqpoint{0.775cm}{1.696cm}}
\pgfpathcurveto{\pgfqpoint{0.801cm}{1.722cm}}{\pgfqpoint{0.815cm}{1.756cm}}{\pgfqpoint{0.815cm}{1.793cm}}
\pgfusepath{fill}
\pgfpathmoveto{\pgfqpoint{1.345cm}{1.765cm}}
\pgfpathcurveto{\pgfqpoint{1.345cm}{1.801cm}}{\pgfqpoint{1.331cm}{1.836cm}}{\pgfqpoint{1.305cm}{1.862cm}}
\pgfpathcurveto{\pgfqpoint{1.28cm}{1.887cm}}{\pgfqpoint{1.245cm}{1.902cm}}{\pgfqpoint{1.209cm}{1.902cm}}
\pgfpathcurveto{\pgfqpoint{1.172cm}{1.902cm}}{\pgfqpoint{1.138cm}{1.887cm}}{\pgfqpoint{1.112cm}{1.862cm}}
\pgfpathcurveto{\pgfqpoint{1.087cm}{1.836cm}}{\pgfqpoint{1.072cm}{1.801cm}}{\pgfqpoint{1.072cm}{1.765cm}}
\pgfpathcurveto{\pgfqpoint{1.072cm}{1.728cm}}{\pgfqpoint{1.087cm}{1.694cm}}{\pgfqpoint{1.112cm}{1.668cm}}
\pgfpathcurveto{\pgfqpoint{1.138cm}{1.642cm}}{\pgfqpoint{1.172cm}{1.628cm}}{\pgfqpoint{1.209cm}{1.628cm}}
\pgfpathcurveto{\pgfqpoint{1.245cm}{1.628cm}}{\pgfqpoint{1.28cm}{1.642cm}}{\pgfqpoint{1.305cm}{1.668cm}}
\pgfpathcurveto{\pgfqpoint{1.331cm}{1.694cm}}{\pgfqpoint{1.345cm}{1.728cm}}{\pgfqpoint{1.345cm}{1.765cm}}
\pgfusepath{fill}
\begin{pgfscope}
\pgfsetdash{}{0cm}
\pgfsetlinewidth{0.818mm}
\pgfsetroundcap
\pgfsetroundjoin
\pgfsetmiterlimit{7.0}
\pgfpathmoveto{\pgfqpoint{0.682cm}{1.065cm}}
\pgfpathlineto{\pgfqpoint{1.246cm}{0.315cm}}
\pgfpathlineto{\pgfqpoint{1.811cm}{1.065cm}}
\pgfusepath{stroke}
\end{pgfscope}
\pgfpathmoveto{\pgfqpoint{1.948cm}{1.065cm}}
\pgfpathcurveto{\pgfqpoint{1.948cm}{1.101cm}}{\pgfqpoint{1.933cm}{1.136cm}}{\pgfqpoint{1.907cm}{1.162cm}}
\pgfpathcurveto{\pgfqpoint{1.882cm}{1.187cm}}{\pgfqpoint{1.847cm}{1.202cm}}{\pgfqpoint{1.811cm}{1.202cm}}
\pgfpathcurveto{\pgfqpoint{1.775cm}{1.202cm}}{\pgfqpoint{1.74cm}{1.187cm}}{\pgfqpoint{1.714cm}{1.162cm}}
\pgfpathcurveto{\pgfqpoint{1.689cm}{1.136cm}}{\pgfqpoint{1.674cm}{1.101cm}}{\pgfqpoint{1.674cm}{1.065cm}}
\pgfpathcurveto{\pgfqpoint{1.674cm}{1.029cm}}{\pgfqpoint{1.689cm}{0.994cm}}{\pgfqpoint{1.714cm}{0.968cm}}
\pgfpathcurveto{\pgfqpoint{1.74cm}{0.942cm}}{\pgfqpoint{1.775cm}{0.928cm}}{\pgfqpoint{1.811cm}{0.928cm}}
\pgfpathcurveto{\pgfqpoint{1.847cm}{0.928cm}}{\pgfqpoint{1.882cm}{0.942cm}}{\pgfqpoint{1.907cm}{0.968cm}}
\pgfpathcurveto{\pgfqpoint{1.933cm}{0.994cm}}{\pgfqpoint{1.948cm}{1.029cm}}{\pgfqpoint{1.948cm}{1.065cm}}
\pgfusepath{fill}
\begin{pgfscope}
\pgfsetdash{}{0cm}
\pgfsetlinewidth{0.818mm}
\pgfsetmiterlimit{7.0}
\pgfpathmoveto{\pgfqpoint{1.246cm}{0.315cm}}
\pgfpathlineto{\pgfqpoint{1.244cm}{1.061cm}}
\pgfusepath{stroke}
\end{pgfscope}
\pgfpathmoveto{\pgfqpoint{1.38cm}{1.065cm}}
\pgfpathcurveto{\pgfqpoint{1.38cm}{1.101cm}}{\pgfqpoint{1.366cm}{1.136cm}}{\pgfqpoint{1.34cm}{1.162cm}}
\pgfpathcurveto{\pgfqpoint{1.315cm}{1.187cm}}{\pgfqpoint{1.28cm}{1.202cm}}{\pgfqpoint{1.244cm}{1.202cm}}
\pgfpathcurveto{\pgfqpoint{1.207cm}{1.202cm}}{\pgfqpoint{1.173cm}{1.187cm}}{\pgfqpoint{1.147cm}{1.162cm}}
\pgfpathcurveto{\pgfqpoint{1.121cm}{1.136cm}}{\pgfqpoint{1.107cm}{1.101cm}}{\pgfqpoint{1.107cm}{1.065cm}}
\pgfpathcurveto{\pgfqpoint{1.107cm}{1.029cm}}{\pgfqpoint{1.121cm}{0.994cm}}{\pgfqpoint{1.147cm}{0.968cm}}
\pgfpathcurveto{\pgfqpoint{1.173cm}{0.942cm}}{\pgfqpoint{1.207cm}{0.928cm}}{\pgfqpoint{1.244cm}{0.928cm}}
\pgfpathcurveto{\pgfqpoint{1.28cm}{0.928cm}}{\pgfqpoint{1.315cm}{0.942cm}}{\pgfqpoint{1.34cm}{0.968cm}}
\pgfpathcurveto{\pgfqpoint{1.366cm}{0.994cm}}{\pgfqpoint{1.38cm}{1.029cm}}{\pgfqpoint{1.38cm}{1.065cm}}
\pgfusepath{fill}
\begin{pgfscope}
\pgfsetdash{}{0cm}
\pgfsetlinewidth{0.818mm}
\pgfsetmiterlimit{4.0}
\pgfpathmoveto{\pgfqpoint{1.383cm}{0.178cm}}
\pgfpathcurveto{\pgfqpoint{1.383cm}{0.214cm}}{\pgfqpoint{1.369cm}{0.249cm}}{\pgfqpoint{1.343cm}{0.275cm}}
\pgfpathcurveto{\pgfqpoint{1.317cm}{0.3cm}}{\pgfqpoint{1.283cm}{0.315cm}}{\pgfqpoint{1.246cm}{0.315cm}}
\pgfpathcurveto{\pgfqpoint{1.21cm}{0.315cm}}{\pgfqpoint{1.175cm}{0.3cm}}{\pgfqpoint{1.15cm}{0.275cm}}
\pgfpathcurveto{\pgfqpoint{1.124cm}{0.249cm}}{\pgfqpoint{1.11cm}{0.214cm}}{\pgfqpoint{1.11cm}{0.178cm}}
\pgfpathcurveto{\pgfqpoint{1.11cm}{0.141cm}}{\pgfqpoint{1.124cm}{0.107cm}}{\pgfqpoint{1.15cm}{0.081cm}}
\pgfpathcurveto{\pgfqpoint{1.175cm}{0.055cm}}{\pgfqpoint{1.21cm}{0.041cm}}{\pgfqpoint{1.246cm}{0.041cm}}
\pgfpathcurveto{\pgfqpoint{1.283cm}{0.041cm}}{\pgfqpoint{1.317cm}{0.055cm}}{\pgfqpoint{1.343cm}{0.081cm}}
\pgfpathcurveto{\pgfqpoint{1.369cm}{0.107cm}}{\pgfqpoint{1.383cm}{0.141cm}}{\pgfqpoint{1.383cm}{0.178cm}}
\pgfusepath{stroke}
\end{pgfscope}
\end{pgfscope}
\end{pgfscope}
\end{pgfscope}
\end{tikzpicture}}}
\end{equation}
and their periodic versions
\begin{equation}\label{eq:r45M}
X_{M},\llbracket X_{M}^{2}\rrbracket,\llbracket X_{M}^{3}\rrbracket, X^{\!\resizebox{0.6em}{!}{
\begin{tikzpicture}
\pgfpathmoveto{\pgfqpoint{0cm}{-0.035cm}}
\pgfpathlineto{\pgfqpoint{1.376cm}{-0.035cm}}
\pgfpathlineto{\pgfqpoint{1.376cm}{1.552cm}}
\pgfpathlineto{\pgfqpoint{0cm}{1.552cm}}
\pgfpathclose
\pgfusepath{clip}
\begin{pgfscope}
\begin{pgfscope}
\pgfpathmoveto{\pgfqpoint{0cm}{-0.035cm}}
\pgfpathlineto{\pgfqpoint{1.376cm}{-0.035cm}}
\pgfpathlineto{\pgfqpoint{1.376cm}{1.552cm}}
\pgfpathlineto{\pgfqpoint{0cm}{1.552cm}}
\pgfpathclose
\pgfusepath{clip}
\begin{pgfscope}
\begin{pgfscope}
\pgfsetdash{}{0cm}
\pgfsetlinewidth{0.818mm}
\pgfsetroundcap
\pgfsetroundjoin
\pgfsetmiterlimit{7.0}
\definecolor{eps2pgf_color}{gray}{0}\pgfsetstrokecolor{eps2pgf_color}\pgfsetfillcolor{eps2pgf_color}
\pgfpathmoveto{\pgfqpoint{0.117cm}{1.421cm}}
\pgfpathlineto{\pgfqpoint{0.682cm}{0.671cm}}
\pgfpathlineto{\pgfqpoint{1.246cm}{1.421cm}}
\pgfusepath{stroke}
\end{pgfscope}
\definecolor{eps2pgf_color}{gray}{0}\pgfsetstrokecolor{eps2pgf_color}\pgfsetfillcolor{eps2pgf_color}
\pgfpathmoveto{\pgfqpoint{0.273cm}{1.395cm}}
\pgfpathcurveto{\pgfqpoint{0.273cm}{1.432cm}}{\pgfqpoint{0.259cm}{1.467cm}}{\pgfqpoint{0.233cm}{1.492cm}}
\pgfpathcurveto{\pgfqpoint{0.207cm}{1.518cm}}{\pgfqpoint{0.173cm}{1.532cm}}{\pgfqpoint{0.137cm}{1.532cm}}
\pgfpathcurveto{\pgfqpoint{0.1cm}{1.532cm}}{\pgfqpoint{0.066cm}{1.518cm}}{\pgfqpoint{0.04cm}{1.492cm}}
\pgfpathcurveto{\pgfqpoint{0.014cm}{1.467cm}}{\pgfqpoint{0cm}{1.432cm}}{\pgfqpoint{0cm}{1.395cm}}
\pgfpathcurveto{\pgfqpoint{0cm}{1.359cm}}{\pgfqpoint{0.014cm}{1.324cm}}{\pgfqpoint{0.04cm}{1.299cm}}
\pgfpathcurveto{\pgfqpoint{0.066cm}{1.273cm}}{\pgfqpoint{0.1cm}{1.258cm}}{\pgfqpoint{0.137cm}{1.258cm}}
\pgfpathcurveto{\pgfqpoint{0.173cm}{1.258cm}}{\pgfqpoint{0.207cm}{1.273cm}}{\pgfqpoint{0.233cm}{1.299cm}}
\pgfpathcurveto{\pgfqpoint{0.259cm}{1.324cm}}{\pgfqpoint{0.273cm}{1.359cm}}{\pgfqpoint{0.273cm}{1.395cm}}
\pgfusepath{fill}
\begin{pgfscope}
\pgfsetdash{}{0cm}
\pgfsetlinewidth{0.818mm}
\pgfsetmiterlimit{7.0}
\pgfpathmoveto{\pgfqpoint{0.682cm}{0.671cm}}
\pgfpathlineto{\pgfqpoint{0.679cm}{1.418cm}}
\pgfusepath{stroke}
\end{pgfscope}
\pgfpathmoveto{\pgfqpoint{0.815cm}{1.399cm}}
\pgfpathcurveto{\pgfqpoint{0.815cm}{1.435cm}}{\pgfqpoint{0.801cm}{1.47cm}}{\pgfqpoint{0.775cm}{1.496cm}}
\pgfpathcurveto{\pgfqpoint{0.75cm}{1.521cm}}{\pgfqpoint{0.715cm}{1.536cm}}{\pgfqpoint{0.679cm}{1.536cm}}
\pgfpathcurveto{\pgfqpoint{0.643cm}{1.536cm}}{\pgfqpoint{0.608cm}{1.521cm}}{\pgfqpoint{0.582cm}{1.496cm}}
\pgfpathcurveto{\pgfqpoint{0.557cm}{1.47cm}}{\pgfqpoint{0.542cm}{1.435cm}}{\pgfqpoint{0.542cm}{1.399cm}}
\pgfpathcurveto{\pgfqpoint{0.542cm}{1.363cm}}{\pgfqpoint{0.557cm}{1.328cm}}{\pgfqpoint{0.582cm}{1.302cm}}
\pgfpathcurveto{\pgfqpoint{0.608cm}{1.276cm}}{\pgfqpoint{0.643cm}{1.262cm}}{\pgfqpoint{0.679cm}{1.262cm}}
\pgfpathcurveto{\pgfqpoint{0.715cm}{1.262cm}}{\pgfqpoint{0.75cm}{1.276cm}}{\pgfqpoint{0.775cm}{1.302cm}}
\pgfpathcurveto{\pgfqpoint{0.801cm}{1.328cm}}{\pgfqpoint{0.815cm}{1.363cm}}{\pgfqpoint{0.815cm}{1.399cm}}
\pgfusepath{fill}
\pgfpathmoveto{\pgfqpoint{1.345cm}{1.371cm}}
\pgfpathcurveto{\pgfqpoint{1.345cm}{1.408cm}}{\pgfqpoint{1.331cm}{1.442cm}}{\pgfqpoint{1.305cm}{1.468cm}}
\pgfpathcurveto{\pgfqpoint{1.28cm}{1.494cm}}{\pgfqpoint{1.245cm}{1.508cm}}{\pgfqpoint{1.209cm}{1.508cm}}
\pgfpathcurveto{\pgfqpoint{1.172cm}{1.508cm}}{\pgfqpoint{1.138cm}{1.494cm}}{\pgfqpoint{1.112cm}{1.468cm}}
\pgfpathcurveto{\pgfqpoint{1.087cm}{1.442cm}}{\pgfqpoint{1.072cm}{1.408cm}}{\pgfqpoint{1.072cm}{1.371cm}}
\pgfpathcurveto{\pgfqpoint{1.072cm}{1.335cm}}{\pgfqpoint{1.087cm}{1.3cm}}{\pgfqpoint{1.112cm}{1.274cm}}
\pgfpathcurveto{\pgfqpoint{1.138cm}{1.249cm}}{\pgfqpoint{1.172cm}{1.234cm}}{\pgfqpoint{1.209cm}{1.234cm}}
\pgfpathcurveto{\pgfqpoint{1.245cm}{1.234cm}}{\pgfqpoint{1.28cm}{1.249cm}}{\pgfqpoint{1.305cm}{1.274cm}}
\pgfpathcurveto{\pgfqpoint{1.331cm}{1.3cm}}{\pgfqpoint{1.345cm}{1.335cm}}{\pgfqpoint{1.345cm}{1.371cm}}
\pgfusepath{fill}
\begin{pgfscope}
\pgfsetdash{}{0cm}
\pgfsetlinewidth{0.818mm}
\pgfsetroundcap
\pgfsetmiterlimit{4.0}
\pgfpathmoveto{\pgfqpoint{0.682cm}{0.671cm}}
\pgfpathlineto{\pgfqpoint{0.682cm}{0.042cm}}
\pgfusepath{stroke}
\end{pgfscope}
\end{pgfscope}
\end{pgfscope}
\end{pgfscope}
\end{tikzpicture}}}_{M}, X^{\!\resizebox{0.6em}{!}{
\begin{tikzpicture}
\pgfpathmoveto{\pgfqpoint{0cm}{0cm}}
\pgfpathlineto{\pgfqpoint{1.376cm}{0cm}}
\pgfpathlineto{\pgfqpoint{1.376cm}{1.588cm}}
\pgfpathlineto{\pgfqpoint{0cm}{1.588cm}}
\pgfpathclose
\pgfusepath{clip}
\begin{pgfscope}
\begin{pgfscope}
\pgfpathmoveto{\pgfqpoint{0cm}{0cm}}
\pgfpathlineto{\pgfqpoint{1.376cm}{0cm}}
\pgfpathlineto{\pgfqpoint{1.376cm}{1.588cm}}
\pgfpathlineto{\pgfqpoint{0cm}{1.588cm}}
\pgfpathclose
\pgfusepath{clip}
\begin{pgfscope}
\begin{pgfscope}
\definecolor{eps2pgf_color}{gray}{0.976471}\pgfsetstrokecolor{eps2pgf_color}\pgfsetfillcolor{eps2pgf_color}
\pgfpathmoveto{\pgfqpoint{0cm}{0cm}}
\pgfpathlineto{\pgfqpoint{1.376cm}{0cm}}
\pgfpathlineto{\pgfqpoint{1.376cm}{1.588cm}}
\pgfpathlineto{\pgfqpoint{0cm}{1.588cm}}
\pgfpathclose
\pgfusepath{fill}
\end{pgfscope}
\begin{pgfscope}
\pgfsetdash{}{0cm}
\pgfsetlinewidth{0.818mm}
\pgfsetroundcap
\pgfsetroundjoin
\pgfsetmiterlimit{7.0}
\definecolor{eps2pgf_color}{gray}{0}\pgfsetstrokecolor{eps2pgf_color}\pgfsetfillcolor{eps2pgf_color}
\pgfpathmoveto{\pgfqpoint{0.117cm}{1.476cm}}
\pgfpathlineto{\pgfqpoint{0.682cm}{0.726cm}}
\pgfpathlineto{\pgfqpoint{1.246cm}{1.476cm}}
\pgfusepath{stroke}
\end{pgfscope}
\definecolor{eps2pgf_color}{gray}{0}\pgfsetstrokecolor{eps2pgf_color}\pgfsetfillcolor{eps2pgf_color}
\pgfpathmoveto{\pgfqpoint{0.273cm}{1.451cm}}
\pgfpathcurveto{\pgfqpoint{0.273cm}{1.487cm}}{\pgfqpoint{0.259cm}{1.522cm}}{\pgfqpoint{0.233cm}{1.547cm}}
\pgfpathcurveto{\pgfqpoint{0.207cm}{1.573cm}}{\pgfqpoint{0.173cm}{1.588cm}}{\pgfqpoint{0.137cm}{1.588cm}}
\pgfpathcurveto{\pgfqpoint{0.1cm}{1.588cm}}{\pgfqpoint{0.066cm}{1.573cm}}{\pgfqpoint{0.04cm}{1.547cm}}
\pgfpathcurveto{\pgfqpoint{0.014cm}{1.522cm}}{\pgfqpoint{0cm}{1.487cm}}{\pgfqpoint{0cm}{1.451cm}}
\pgfpathcurveto{\pgfqpoint{0cm}{1.414cm}}{\pgfqpoint{0.014cm}{1.379cm}}{\pgfqpoint{0.04cm}{1.354cm}}
\pgfpathcurveto{\pgfqpoint{0.066cm}{1.328cm}}{\pgfqpoint{0.1cm}{1.314cm}}{\pgfqpoint{0.137cm}{1.314cm}}
\pgfpathcurveto{\pgfqpoint{0.173cm}{1.314cm}}{\pgfqpoint{0.207cm}{1.328cm}}{\pgfqpoint{0.233cm}{1.354cm}}
\pgfpathcurveto{\pgfqpoint{0.259cm}{1.379cm}}{\pgfqpoint{0.273cm}{1.414cm}}{\pgfqpoint{0.273cm}{1.451cm}}
\pgfusepath{fill}
\pgfpathmoveto{\pgfqpoint{1.345cm}{1.426cm}}
\pgfpathcurveto{\pgfqpoint{1.345cm}{1.463cm}}{\pgfqpoint{1.331cm}{1.497cm}}{\pgfqpoint{1.305cm}{1.523cm}}
\pgfpathcurveto{\pgfqpoint{1.28cm}{1.549cm}}{\pgfqpoint{1.245cm}{1.563cm}}{\pgfqpoint{1.209cm}{1.563cm}}
\pgfpathcurveto{\pgfqpoint{1.172cm}{1.563cm}}{\pgfqpoint{1.138cm}{1.549cm}}{\pgfqpoint{1.112cm}{1.523cm}}
\pgfpathcurveto{\pgfqpoint{1.087cm}{1.497cm}}{\pgfqpoint{1.072cm}{1.463cm}}{\pgfqpoint{1.072cm}{1.426cm}}
\pgfpathcurveto{\pgfqpoint{1.072cm}{1.39cm}}{\pgfqpoint{1.087cm}{1.355cm}}{\pgfqpoint{1.112cm}{1.329cm}}
\pgfpathcurveto{\pgfqpoint{1.138cm}{1.304cm}}{\pgfqpoint{1.172cm}{1.289cm}}{\pgfqpoint{1.209cm}{1.289cm}}
\pgfpathcurveto{\pgfqpoint{1.245cm}{1.289cm}}{\pgfqpoint{1.28cm}{1.304cm}}{\pgfqpoint{1.305cm}{1.329cm}}
\pgfpathcurveto{\pgfqpoint{1.331cm}{1.355cm}}{\pgfqpoint{1.345cm}{1.39cm}}{\pgfqpoint{1.345cm}{1.426cm}}
\pgfusepath{fill}
\begin{pgfscope}
\pgfsetdash{}{0cm}
\pgfsetlinewidth{0.818mm}
\pgfsetroundcap
\pgfsetmiterlimit{4.0}
\pgfpathmoveto{\pgfqpoint{0.682cm}{0.726cm}}
\pgfpathlineto{\pgfqpoint{0.682cm}{0.097cm}}
\pgfusepath{stroke}
\end{pgfscope}
\end{pgfscope}
\end{pgfscope}
\end{pgfscope}
\end{tikzpicture}}}_{M}, X^{\!\resizebox{!}{.8em}{
\begin{tikzpicture}
\pgfpathmoveto{\pgfqpoint{0cm}{-0.035cm}}
\pgfpathlineto{\pgfqpoint{1.976cm}{-0.035cm}}
\pgfpathlineto{\pgfqpoint{1.976cm}{1.94cm}}
\pgfpathlineto{\pgfqpoint{0cm}{1.94cm}}
\pgfpathclose
\pgfusepath{clip}
\begin{pgfscope}
\begin{pgfscope}
\pgfpathmoveto{\pgfqpoint{0cm}{-0.035cm}}
\pgfpathlineto{\pgfqpoint{1.976cm}{-0.035cm}}
\pgfpathlineto{\pgfqpoint{1.976cm}{1.94cm}}
\pgfpathlineto{\pgfqpoint{0cm}{1.94cm}}
\pgfpathclose
\pgfusepath{clip}
\begin{pgfscope}
\begin{pgfscope}
\pgfsetdash{}{0cm}
\pgfsetlinewidth{0.818mm}
\pgfsetroundcap
\pgfsetroundjoin
\pgfsetmiterlimit{7.0}
\definecolor{eps2pgf_color}{gray}{0}\pgfsetstrokecolor{eps2pgf_color}\pgfsetfillcolor{eps2pgf_color}
\pgfpathmoveto{\pgfqpoint{0.117cm}{1.815cm}}
\pgfpathlineto{\pgfqpoint{0.682cm}{1.065cm}}
\pgfpathlineto{\pgfqpoint{1.246cm}{1.815cm}}
\pgfusepath{stroke}
\end{pgfscope}
\definecolor{eps2pgf_color}{gray}{0}\pgfsetstrokecolor{eps2pgf_color}\pgfsetfillcolor{eps2pgf_color}
\pgfpathmoveto{\pgfqpoint{0.273cm}{1.789cm}}
\pgfpathcurveto{\pgfqpoint{0.273cm}{1.825cm}}{\pgfqpoint{0.259cm}{1.86cm}}{\pgfqpoint{0.233cm}{1.886cm}}
\pgfpathcurveto{\pgfqpoint{0.207cm}{1.912cm}}{\pgfqpoint{0.173cm}{1.926cm}}{\pgfqpoint{0.137cm}{1.926cm}}
\pgfpathcurveto{\pgfqpoint{0.1cm}{1.926cm}}{\pgfqpoint{0.066cm}{1.912cm}}{\pgfqpoint{0.04cm}{1.886cm}}
\pgfpathcurveto{\pgfqpoint{0.014cm}{1.86cm}}{\pgfqpoint{0cm}{1.825cm}}{\pgfqpoint{0cm}{1.789cm}}
\pgfpathcurveto{\pgfqpoint{0cm}{1.753cm}}{\pgfqpoint{0.014cm}{1.718cm}}{\pgfqpoint{0.04cm}{1.692cm}}
\pgfpathcurveto{\pgfqpoint{0.066cm}{1.667cm}}{\pgfqpoint{0.1cm}{1.652cm}}{\pgfqpoint{0.137cm}{1.652cm}}
\pgfpathcurveto{\pgfqpoint{0.173cm}{1.652cm}}{\pgfqpoint{0.207cm}{1.667cm}}{\pgfqpoint{0.233cm}{1.692cm}}
\pgfpathcurveto{\pgfqpoint{0.259cm}{1.718cm}}{\pgfqpoint{0.273cm}{1.753cm}}{\pgfqpoint{0.273cm}{1.789cm}}
\pgfusepath{fill}
\begin{pgfscope}
\pgfsetdash{}{0cm}
\pgfsetlinewidth{0.818mm}
\pgfsetmiterlimit{7.0}
\pgfpathmoveto{\pgfqpoint{0.682cm}{1.065cm}}
\pgfpathlineto{\pgfqpoint{0.679cm}{1.812cm}}
\pgfusepath{stroke}
\end{pgfscope}
\pgfpathmoveto{\pgfqpoint{0.815cm}{1.793cm}}
\pgfpathcurveto{\pgfqpoint{0.815cm}{1.829cm}}{\pgfqpoint{0.801cm}{1.864cm}}{\pgfqpoint{0.775cm}{1.89cm}}
\pgfpathcurveto{\pgfqpoint{0.75cm}{1.915cm}}{\pgfqpoint{0.715cm}{1.93cm}}{\pgfqpoint{0.679cm}{1.93cm}}
\pgfpathcurveto{\pgfqpoint{0.643cm}{1.93cm}}{\pgfqpoint{0.608cm}{1.915cm}}{\pgfqpoint{0.582cm}{1.89cm}}
\pgfpathcurveto{\pgfqpoint{0.557cm}{1.864cm}}{\pgfqpoint{0.542cm}{1.829cm}}{\pgfqpoint{0.542cm}{1.793cm}}
\pgfpathcurveto{\pgfqpoint{0.542cm}{1.756cm}}{\pgfqpoint{0.557cm}{1.722cm}}{\pgfqpoint{0.582cm}{1.696cm}}
\pgfpathcurveto{\pgfqpoint{0.608cm}{1.67cm}}{\pgfqpoint{0.643cm}{1.656cm}}{\pgfqpoint{0.679cm}{1.656cm}}
\pgfpathcurveto{\pgfqpoint{0.715cm}{1.656cm}}{\pgfqpoint{0.75cm}{1.67cm}}{\pgfqpoint{0.775cm}{1.696cm}}
\pgfpathcurveto{\pgfqpoint{0.801cm}{1.722cm}}{\pgfqpoint{0.815cm}{1.756cm}}{\pgfqpoint{0.815cm}{1.793cm}}
\pgfusepath{fill}
\pgfpathmoveto{\pgfqpoint{1.345cm}{1.765cm}}
\pgfpathcurveto{\pgfqpoint{1.345cm}{1.801cm}}{\pgfqpoint{1.331cm}{1.836cm}}{\pgfqpoint{1.305cm}{1.862cm}}
\pgfpathcurveto{\pgfqpoint{1.28cm}{1.887cm}}{\pgfqpoint{1.245cm}{1.902cm}}{\pgfqpoint{1.209cm}{1.902cm}}
\pgfpathcurveto{\pgfqpoint{1.172cm}{1.902cm}}{\pgfqpoint{1.138cm}{1.887cm}}{\pgfqpoint{1.112cm}{1.862cm}}
\pgfpathcurveto{\pgfqpoint{1.087cm}{1.836cm}}{\pgfqpoint{1.072cm}{1.801cm}}{\pgfqpoint{1.072cm}{1.765cm}}
\pgfpathcurveto{\pgfqpoint{1.072cm}{1.728cm}}{\pgfqpoint{1.087cm}{1.694cm}}{\pgfqpoint{1.112cm}{1.668cm}}
\pgfpathcurveto{\pgfqpoint{1.138cm}{1.642cm}}{\pgfqpoint{1.172cm}{1.628cm}}{\pgfqpoint{1.209cm}{1.628cm}}
\pgfpathcurveto{\pgfqpoint{1.245cm}{1.628cm}}{\pgfqpoint{1.28cm}{1.642cm}}{\pgfqpoint{1.305cm}{1.668cm}}
\pgfpathcurveto{\pgfqpoint{1.331cm}{1.694cm}}{\pgfqpoint{1.345cm}{1.728cm}}{\pgfqpoint{1.345cm}{1.765cm}}
\pgfusepath{fill}
\begin{pgfscope}
\pgfsetdash{}{0cm}
\pgfsetlinewidth{0.818mm}
\pgfsetroundcap
\pgfsetroundjoin
\pgfsetmiterlimit{7.0}
\pgfpathmoveto{\pgfqpoint{0.682cm}{1.065cm}}
\pgfpathlineto{\pgfqpoint{1.246cm}{0.315cm}}
\pgfpathlineto{\pgfqpoint{1.811cm}{1.065cm}}
\pgfusepath{stroke}
\end{pgfscope}
\pgfpathmoveto{\pgfqpoint{1.948cm}{1.065cm}}
\pgfpathcurveto{\pgfqpoint{1.948cm}{1.101cm}}{\pgfqpoint{1.933cm}{1.136cm}}{\pgfqpoint{1.907cm}{1.162cm}}
\pgfpathcurveto{\pgfqpoint{1.882cm}{1.187cm}}{\pgfqpoint{1.847cm}{1.202cm}}{\pgfqpoint{1.811cm}{1.202cm}}
\pgfpathcurveto{\pgfqpoint{1.775cm}{1.202cm}}{\pgfqpoint{1.74cm}{1.187cm}}{\pgfqpoint{1.714cm}{1.162cm}}
\pgfpathcurveto{\pgfqpoint{1.689cm}{1.136cm}}{\pgfqpoint{1.674cm}{1.101cm}}{\pgfqpoint{1.674cm}{1.065cm}}
\pgfpathcurveto{\pgfqpoint{1.674cm}{1.029cm}}{\pgfqpoint{1.689cm}{0.994cm}}{\pgfqpoint{1.714cm}{0.968cm}}
\pgfpathcurveto{\pgfqpoint{1.74cm}{0.942cm}}{\pgfqpoint{1.775cm}{0.928cm}}{\pgfqpoint{1.811cm}{0.928cm}}
\pgfpathcurveto{\pgfqpoint{1.847cm}{0.928cm}}{\pgfqpoint{1.882cm}{0.942cm}}{\pgfqpoint{1.907cm}{0.968cm}}
\pgfpathcurveto{\pgfqpoint{1.933cm}{0.994cm}}{\pgfqpoint{1.948cm}{1.029cm}}{\pgfqpoint{1.948cm}{1.065cm}}
\pgfusepath{fill}
\begin{pgfscope}
\pgfsetdash{}{0cm}
\pgfsetlinewidth{0.818mm}
\pgfsetmiterlimit{4.0}
\pgfpathmoveto{\pgfqpoint{1.383cm}{0.178cm}}
\pgfpathcurveto{\pgfqpoint{1.383cm}{0.214cm}}{\pgfqpoint{1.369cm}{0.249cm}}{\pgfqpoint{1.343cm}{0.275cm}}
\pgfpathcurveto{\pgfqpoint{1.317cm}{0.3cm}}{\pgfqpoint{1.283cm}{0.315cm}}{\pgfqpoint{1.246cm}{0.315cm}}
\pgfpathcurveto{\pgfqpoint{1.21cm}{0.315cm}}{\pgfqpoint{1.175cm}{0.3cm}}{\pgfqpoint{1.15cm}{0.275cm}}
\pgfpathcurveto{\pgfqpoint{1.124cm}{0.249cm}}{\pgfqpoint{1.11cm}{0.214cm}}{\pgfqpoint{1.11cm}{0.178cm}}
\pgfpathcurveto{\pgfqpoint{1.11cm}{0.141cm}}{\pgfqpoint{1.124cm}{0.107cm}}{\pgfqpoint{1.15cm}{0.081cm}}
\pgfpathcurveto{\pgfqpoint{1.175cm}{0.055cm}}{\pgfqpoint{1.21cm}{0.041cm}}{\pgfqpoint{1.246cm}{0.041cm}}
\pgfpathcurveto{\pgfqpoint{1.283cm}{0.041cm}}{\pgfqpoint{1.317cm}{0.055cm}}{\pgfqpoint{1.343cm}{0.081cm}}
\pgfpathcurveto{\pgfqpoint{1.369cm}{0.107cm}}{\pgfqpoint{1.383cm}{0.141cm}}{\pgfqpoint{1.383cm}{0.178cm}}
\pgfusepath{stroke}
\end{pgfscope}
\end{pgfscope}
\end{pgfscope}
\end{pgfscope}
\end{tikzpicture}}}_{M}, X^{\!\resizebox{!}{.8em}{
\begin{tikzpicture}
\pgfpathmoveto{\pgfqpoint{0cm}{-0.035cm}}
\pgfpathlineto{\pgfqpoint{1.976cm}{-0.035cm}}
\pgfpathlineto{\pgfqpoint{1.976cm}{1.94cm}}
\pgfpathlineto{\pgfqpoint{0cm}{1.94cm}}
\pgfpathclose
\pgfusepath{clip}
\begin{pgfscope}
\begin{pgfscope}
\pgfpathmoveto{\pgfqpoint{0cm}{-0.035cm}}
\pgfpathlineto{\pgfqpoint{1.976cm}{-0.035cm}}
\pgfpathlineto{\pgfqpoint{1.976cm}{1.94cm}}
\pgfpathlineto{\pgfqpoint{0cm}{1.94cm}}
\pgfpathclose
\pgfusepath{clip}
\begin{pgfscope}
\begin{pgfscope}
\pgfsetdash{}{0cm}
\pgfsetlinewidth{0.818mm}
\pgfsetroundcap
\pgfsetroundjoin
\pgfsetmiterlimit{7.0}
\definecolor{eps2pgf_color}{gray}{0}\pgfsetstrokecolor{eps2pgf_color}\pgfsetfillcolor{eps2pgf_color}
\pgfpathmoveto{\pgfqpoint{0.117cm}{1.815cm}}
\pgfpathlineto{\pgfqpoint{0.682cm}{1.065cm}}
\pgfpathlineto{\pgfqpoint{1.246cm}{1.815cm}}
\pgfusepath{stroke}
\end{pgfscope}
\definecolor{eps2pgf_color}{gray}{0}\pgfsetstrokecolor{eps2pgf_color}\pgfsetfillcolor{eps2pgf_color}
\pgfpathmoveto{\pgfqpoint{0.273cm}{1.789cm}}
\pgfpathcurveto{\pgfqpoint{0.273cm}{1.825cm}}{\pgfqpoint{0.259cm}{1.86cm}}{\pgfqpoint{0.233cm}{1.886cm}}
\pgfpathcurveto{\pgfqpoint{0.207cm}{1.912cm}}{\pgfqpoint{0.173cm}{1.926cm}}{\pgfqpoint{0.137cm}{1.926cm}}
\pgfpathcurveto{\pgfqpoint{0.1cm}{1.926cm}}{\pgfqpoint{0.066cm}{1.912cm}}{\pgfqpoint{0.04cm}{1.886cm}}
\pgfpathcurveto{\pgfqpoint{0.014cm}{1.86cm}}{\pgfqpoint{0cm}{1.825cm}}{\pgfqpoint{0cm}{1.789cm}}
\pgfpathcurveto{\pgfqpoint{0cm}{1.753cm}}{\pgfqpoint{0.014cm}{1.718cm}}{\pgfqpoint{0.04cm}{1.692cm}}
\pgfpathcurveto{\pgfqpoint{0.066cm}{1.667cm}}{\pgfqpoint{0.1cm}{1.652cm}}{\pgfqpoint{0.137cm}{1.652cm}}
\pgfpathcurveto{\pgfqpoint{0.173cm}{1.652cm}}{\pgfqpoint{0.207cm}{1.667cm}}{\pgfqpoint{0.233cm}{1.692cm}}
\pgfpathcurveto{\pgfqpoint{0.259cm}{1.718cm}}{\pgfqpoint{0.273cm}{1.753cm}}{\pgfqpoint{0.273cm}{1.789cm}}
\pgfusepath{fill}
\pgfpathmoveto{\pgfqpoint{1.345cm}{1.765cm}}
\pgfpathcurveto{\pgfqpoint{1.345cm}{1.801cm}}{\pgfqpoint{1.331cm}{1.836cm}}{\pgfqpoint{1.305cm}{1.862cm}}
\pgfpathcurveto{\pgfqpoint{1.28cm}{1.887cm}}{\pgfqpoint{1.245cm}{1.902cm}}{\pgfqpoint{1.209cm}{1.902cm}}
\pgfpathcurveto{\pgfqpoint{1.172cm}{1.902cm}}{\pgfqpoint{1.138cm}{1.887cm}}{\pgfqpoint{1.112cm}{1.862cm}}
\pgfpathcurveto{\pgfqpoint{1.087cm}{1.836cm}}{\pgfqpoint{1.072cm}{1.801cm}}{\pgfqpoint{1.072cm}{1.765cm}}
\pgfpathcurveto{\pgfqpoint{1.072cm}{1.728cm}}{\pgfqpoint{1.087cm}{1.694cm}}{\pgfqpoint{1.112cm}{1.668cm}}
\pgfpathcurveto{\pgfqpoint{1.138cm}{1.642cm}}{\pgfqpoint{1.172cm}{1.628cm}}{\pgfqpoint{1.209cm}{1.628cm}}
\pgfpathcurveto{\pgfqpoint{1.245cm}{1.628cm}}{\pgfqpoint{1.28cm}{1.642cm}}{\pgfqpoint{1.305cm}{1.668cm}}
\pgfpathcurveto{\pgfqpoint{1.331cm}{1.694cm}}{\pgfqpoint{1.345cm}{1.728cm}}{\pgfqpoint{1.345cm}{1.765cm}}
\pgfusepath{fill}
\begin{pgfscope}
\pgfsetdash{}{0cm}
\pgfsetlinewidth{0.818mm}
\pgfsetroundcap
\pgfsetroundjoin
\pgfsetmiterlimit{7.0}
\pgfpathmoveto{\pgfqpoint{0.682cm}{1.065cm}}
\pgfpathlineto{\pgfqpoint{1.246cm}{0.315cm}}
\pgfpathlineto{\pgfqpoint{1.811cm}{1.065cm}}
\pgfusepath{stroke}
\end{pgfscope}
\pgfpathmoveto{\pgfqpoint{1.948cm}{1.065cm}}
\pgfpathcurveto{\pgfqpoint{1.948cm}{1.101cm}}{\pgfqpoint{1.933cm}{1.136cm}}{\pgfqpoint{1.907cm}{1.162cm}}
\pgfpathcurveto{\pgfqpoint{1.882cm}{1.187cm}}{\pgfqpoint{1.847cm}{1.202cm}}{\pgfqpoint{1.811cm}{1.202cm}}
\pgfpathcurveto{\pgfqpoint{1.775cm}{1.202cm}}{\pgfqpoint{1.74cm}{1.187cm}}{\pgfqpoint{1.714cm}{1.162cm}}
\pgfpathcurveto{\pgfqpoint{1.689cm}{1.136cm}}{\pgfqpoint{1.674cm}{1.101cm}}{\pgfqpoint{1.674cm}{1.065cm}}
\pgfpathcurveto{\pgfqpoint{1.674cm}{1.029cm}}{\pgfqpoint{1.689cm}{0.994cm}}{\pgfqpoint{1.714cm}{0.968cm}}
\pgfpathcurveto{\pgfqpoint{1.74cm}{0.942cm}}{\pgfqpoint{1.775cm}{0.928cm}}{\pgfqpoint{1.811cm}{0.928cm}}
\pgfpathcurveto{\pgfqpoint{1.847cm}{0.928cm}}{\pgfqpoint{1.882cm}{0.942cm}}{\pgfqpoint{1.907cm}{0.968cm}}
\pgfpathcurveto{\pgfqpoint{1.933cm}{0.994cm}}{\pgfqpoint{1.948cm}{1.029cm}}{\pgfqpoint{1.948cm}{1.065cm}}
\pgfusepath{fill}
\begin{pgfscope}
\pgfsetdash{}{0cm}
\pgfsetlinewidth{0.818mm}
\pgfsetmiterlimit{7.0}
\pgfpathmoveto{\pgfqpoint{1.246cm}{0.315cm}}
\pgfpathlineto{\pgfqpoint{1.244cm}{1.061cm}}
\pgfusepath{stroke}
\end{pgfscope}
\pgfpathmoveto{\pgfqpoint{1.38cm}{1.065cm}}
\pgfpathcurveto{\pgfqpoint{1.38cm}{1.101cm}}{\pgfqpoint{1.366cm}{1.136cm}}{\pgfqpoint{1.34cm}{1.162cm}}
\pgfpathcurveto{\pgfqpoint{1.315cm}{1.187cm}}{\pgfqpoint{1.28cm}{1.202cm}}{\pgfqpoint{1.244cm}{1.202cm}}
\pgfpathcurveto{\pgfqpoint{1.207cm}{1.202cm}}{\pgfqpoint{1.173cm}{1.187cm}}{\pgfqpoint{1.147cm}{1.162cm}}
\pgfpathcurveto{\pgfqpoint{1.121cm}{1.136cm}}{\pgfqpoint{1.107cm}{1.101cm}}{\pgfqpoint{1.107cm}{1.065cm}}
\pgfpathcurveto{\pgfqpoint{1.107cm}{1.029cm}}{\pgfqpoint{1.121cm}{0.994cm}}{\pgfqpoint{1.147cm}{0.968cm}}
\pgfpathcurveto{\pgfqpoint{1.173cm}{0.942cm}}{\pgfqpoint{1.207cm}{0.928cm}}{\pgfqpoint{1.244cm}{0.928cm}}
\pgfpathcurveto{\pgfqpoint{1.28cm}{0.928cm}}{\pgfqpoint{1.315cm}{0.942cm}}{\pgfqpoint{1.34cm}{0.968cm}}
\pgfpathcurveto{\pgfqpoint{1.366cm}{0.994cm}}{\pgfqpoint{1.38cm}{1.029cm}}{\pgfqpoint{1.38cm}{1.065cm}}
\pgfusepath{fill}
\begin{pgfscope}
\pgfsetdash{}{0cm}
\pgfsetlinewidth{0.818mm}
\pgfsetmiterlimit{4.0}
\pgfpathmoveto{\pgfqpoint{1.383cm}{0.178cm}}
\pgfpathcurveto{\pgfqpoint{1.383cm}{0.214cm}}{\pgfqpoint{1.369cm}{0.249cm}}{\pgfqpoint{1.343cm}{0.275cm}}
\pgfpathcurveto{\pgfqpoint{1.317cm}{0.3cm}}{\pgfqpoint{1.283cm}{0.315cm}}{\pgfqpoint{1.246cm}{0.315cm}}
\pgfpathcurveto{\pgfqpoint{1.21cm}{0.315cm}}{\pgfqpoint{1.175cm}{0.3cm}}{\pgfqpoint{1.15cm}{0.275cm}}
\pgfpathcurveto{\pgfqpoint{1.124cm}{0.249cm}}{\pgfqpoint{1.11cm}{0.214cm}}{\pgfqpoint{1.11cm}{0.178cm}}
\pgfpathcurveto{\pgfqpoint{1.11cm}{0.141cm}}{\pgfqpoint{1.124cm}{0.107cm}}{\pgfqpoint{1.15cm}{0.081cm}}
\pgfpathcurveto{\pgfqpoint{1.175cm}{0.055cm}}{\pgfqpoint{1.21cm}{0.041cm}}{\pgfqpoint{1.246cm}{0.041cm}}
\pgfpathcurveto{\pgfqpoint{1.283cm}{0.041cm}}{\pgfqpoint{1.317cm}{0.055cm}}{\pgfqpoint{1.343cm}{0.081cm}}
\pgfpathcurveto{\pgfqpoint{1.369cm}{0.107cm}}{\pgfqpoint{1.383cm}{0.141cm}}{\pgfqpoint{1.383cm}{0.178cm}}
\pgfusepath{stroke}
\end{pgfscope}
\end{pgfscope}
\end{pgfscope}
\end{pgfscope}
\end{tikzpicture}}}_{M}, X^{\!\resizebox{!}{.8em}{
\begin{tikzpicture}
\pgfpathmoveto{\pgfqpoint{0cm}{-0.035cm}}
\pgfpathlineto{\pgfqpoint{1.976cm}{-0.035cm}}
\pgfpathlineto{\pgfqpoint{1.976cm}{1.94cm}}
\pgfpathlineto{\pgfqpoint{0cm}{1.94cm}}
\pgfpathclose
\pgfusepath{clip}
\begin{pgfscope}
\begin{pgfscope}
\pgfpathmoveto{\pgfqpoint{0cm}{-0.035cm}}
\pgfpathlineto{\pgfqpoint{1.976cm}{-0.035cm}}
\pgfpathlineto{\pgfqpoint{1.976cm}{1.94cm}}
\pgfpathlineto{\pgfqpoint{0cm}{1.94cm}}
\pgfpathclose
\pgfusepath{clip}
\begin{pgfscope}
\begin{pgfscope}
\pgfsetdash{}{0cm}
\pgfsetlinewidth{0.818mm}
\pgfsetroundcap
\pgfsetroundjoin
\pgfsetmiterlimit{7.0}
\definecolor{eps2pgf_color}{gray}{0}\pgfsetstrokecolor{eps2pgf_color}\pgfsetfillcolor{eps2pgf_color}
\pgfpathmoveto{\pgfqpoint{0.117cm}{1.815cm}}
\pgfpathlineto{\pgfqpoint{0.682cm}{1.065cm}}
\pgfpathlineto{\pgfqpoint{1.246cm}{1.815cm}}
\pgfusepath{stroke}
\end{pgfscope}
\definecolor{eps2pgf_color}{gray}{0}\pgfsetstrokecolor{eps2pgf_color}\pgfsetfillcolor{eps2pgf_color}
\pgfpathmoveto{\pgfqpoint{0.273cm}{1.789cm}}
\pgfpathcurveto{\pgfqpoint{0.273cm}{1.825cm}}{\pgfqpoint{0.259cm}{1.86cm}}{\pgfqpoint{0.233cm}{1.886cm}}
\pgfpathcurveto{\pgfqpoint{0.207cm}{1.912cm}}{\pgfqpoint{0.173cm}{1.926cm}}{\pgfqpoint{0.137cm}{1.926cm}}
\pgfpathcurveto{\pgfqpoint{0.1cm}{1.926cm}}{\pgfqpoint{0.066cm}{1.912cm}}{\pgfqpoint{0.04cm}{1.886cm}}
\pgfpathcurveto{\pgfqpoint{0.014cm}{1.86cm}}{\pgfqpoint{0cm}{1.825cm}}{\pgfqpoint{0cm}{1.789cm}}
\pgfpathcurveto{\pgfqpoint{0cm}{1.753cm}}{\pgfqpoint{0.014cm}{1.718cm}}{\pgfqpoint{0.04cm}{1.692cm}}
\pgfpathcurveto{\pgfqpoint{0.066cm}{1.667cm}}{\pgfqpoint{0.1cm}{1.652cm}}{\pgfqpoint{0.137cm}{1.652cm}}
\pgfpathcurveto{\pgfqpoint{0.173cm}{1.652cm}}{\pgfqpoint{0.207cm}{1.667cm}}{\pgfqpoint{0.233cm}{1.692cm}}
\pgfpathcurveto{\pgfqpoint{0.259cm}{1.718cm}}{\pgfqpoint{0.273cm}{1.753cm}}{\pgfqpoint{0.273cm}{1.789cm}}
\pgfusepath{fill}
\begin{pgfscope}
\pgfsetdash{}{0cm}
\pgfsetlinewidth{0.818mm}
\pgfsetmiterlimit{7.0}
\pgfpathmoveto{\pgfqpoint{0.682cm}{1.065cm}}
\pgfpathlineto{\pgfqpoint{0.679cm}{1.812cm}}
\pgfusepath{stroke}
\end{pgfscope}
\pgfpathmoveto{\pgfqpoint{0.815cm}{1.793cm}}
\pgfpathcurveto{\pgfqpoint{0.815cm}{1.829cm}}{\pgfqpoint{0.801cm}{1.864cm}}{\pgfqpoint{0.775cm}{1.89cm}}
\pgfpathcurveto{\pgfqpoint{0.75cm}{1.915cm}}{\pgfqpoint{0.715cm}{1.93cm}}{\pgfqpoint{0.679cm}{1.93cm}}
\pgfpathcurveto{\pgfqpoint{0.643cm}{1.93cm}}{\pgfqpoint{0.608cm}{1.915cm}}{\pgfqpoint{0.582cm}{1.89cm}}
\pgfpathcurveto{\pgfqpoint{0.557cm}{1.864cm}}{\pgfqpoint{0.542cm}{1.829cm}}{\pgfqpoint{0.542cm}{1.793cm}}
\pgfpathcurveto{\pgfqpoint{0.542cm}{1.756cm}}{\pgfqpoint{0.557cm}{1.722cm}}{\pgfqpoint{0.582cm}{1.696cm}}
\pgfpathcurveto{\pgfqpoint{0.608cm}{1.67cm}}{\pgfqpoint{0.643cm}{1.656cm}}{\pgfqpoint{0.679cm}{1.656cm}}
\pgfpathcurveto{\pgfqpoint{0.715cm}{1.656cm}}{\pgfqpoint{0.75cm}{1.67cm}}{\pgfqpoint{0.775cm}{1.696cm}}
\pgfpathcurveto{\pgfqpoint{0.801cm}{1.722cm}}{\pgfqpoint{0.815cm}{1.756cm}}{\pgfqpoint{0.815cm}{1.793cm}}
\pgfusepath{fill}
\pgfpathmoveto{\pgfqpoint{1.345cm}{1.765cm}}
\pgfpathcurveto{\pgfqpoint{1.345cm}{1.801cm}}{\pgfqpoint{1.331cm}{1.836cm}}{\pgfqpoint{1.305cm}{1.862cm}}
\pgfpathcurveto{\pgfqpoint{1.28cm}{1.887cm}}{\pgfqpoint{1.245cm}{1.902cm}}{\pgfqpoint{1.209cm}{1.902cm}}
\pgfpathcurveto{\pgfqpoint{1.172cm}{1.902cm}}{\pgfqpoint{1.138cm}{1.887cm}}{\pgfqpoint{1.112cm}{1.862cm}}
\pgfpathcurveto{\pgfqpoint{1.087cm}{1.836cm}}{\pgfqpoint{1.072cm}{1.801cm}}{\pgfqpoint{1.072cm}{1.765cm}}
\pgfpathcurveto{\pgfqpoint{1.072cm}{1.728cm}}{\pgfqpoint{1.087cm}{1.694cm}}{\pgfqpoint{1.112cm}{1.668cm}}
\pgfpathcurveto{\pgfqpoint{1.138cm}{1.642cm}}{\pgfqpoint{1.172cm}{1.628cm}}{\pgfqpoint{1.209cm}{1.628cm}}
\pgfpathcurveto{\pgfqpoint{1.245cm}{1.628cm}}{\pgfqpoint{1.28cm}{1.642cm}}{\pgfqpoint{1.305cm}{1.668cm}}
\pgfpathcurveto{\pgfqpoint{1.331cm}{1.694cm}}{\pgfqpoint{1.345cm}{1.728cm}}{\pgfqpoint{1.345cm}{1.765cm}}
\pgfusepath{fill}
\begin{pgfscope}
\pgfsetdash{}{0cm}
\pgfsetlinewidth{0.818mm}
\pgfsetroundcap
\pgfsetroundjoin
\pgfsetmiterlimit{7.0}
\pgfpathmoveto{\pgfqpoint{0.682cm}{1.065cm}}
\pgfpathlineto{\pgfqpoint{1.246cm}{0.315cm}}
\pgfpathlineto{\pgfqpoint{1.811cm}{1.065cm}}
\pgfusepath{stroke}
\end{pgfscope}
\pgfpathmoveto{\pgfqpoint{1.948cm}{1.065cm}}
\pgfpathcurveto{\pgfqpoint{1.948cm}{1.101cm}}{\pgfqpoint{1.933cm}{1.136cm}}{\pgfqpoint{1.907cm}{1.162cm}}
\pgfpathcurveto{\pgfqpoint{1.882cm}{1.187cm}}{\pgfqpoint{1.847cm}{1.202cm}}{\pgfqpoint{1.811cm}{1.202cm}}
\pgfpathcurveto{\pgfqpoint{1.775cm}{1.202cm}}{\pgfqpoint{1.74cm}{1.187cm}}{\pgfqpoint{1.714cm}{1.162cm}}
\pgfpathcurveto{\pgfqpoint{1.689cm}{1.136cm}}{\pgfqpoint{1.674cm}{1.101cm}}{\pgfqpoint{1.674cm}{1.065cm}}
\pgfpathcurveto{\pgfqpoint{1.674cm}{1.029cm}}{\pgfqpoint{1.689cm}{0.994cm}}{\pgfqpoint{1.714cm}{0.968cm}}
\pgfpathcurveto{\pgfqpoint{1.74cm}{0.942cm}}{\pgfqpoint{1.775cm}{0.928cm}}{\pgfqpoint{1.811cm}{0.928cm}}
\pgfpathcurveto{\pgfqpoint{1.847cm}{0.928cm}}{\pgfqpoint{1.882cm}{0.942cm}}{\pgfqpoint{1.907cm}{0.968cm}}
\pgfpathcurveto{\pgfqpoint{1.933cm}{0.994cm}}{\pgfqpoint{1.948cm}{1.029cm}}{\pgfqpoint{1.948cm}{1.065cm}}
\pgfusepath{fill}
\begin{pgfscope}
\pgfsetdash{}{0cm}
\pgfsetlinewidth{0.818mm}
\pgfsetmiterlimit{7.0}
\pgfpathmoveto{\pgfqpoint{1.246cm}{0.315cm}}
\pgfpathlineto{\pgfqpoint{1.244cm}{1.061cm}}
\pgfusepath{stroke}
\end{pgfscope}
\pgfpathmoveto{\pgfqpoint{1.38cm}{1.065cm}}
\pgfpathcurveto{\pgfqpoint{1.38cm}{1.101cm}}{\pgfqpoint{1.366cm}{1.136cm}}{\pgfqpoint{1.34cm}{1.162cm}}
\pgfpathcurveto{\pgfqpoint{1.315cm}{1.187cm}}{\pgfqpoint{1.28cm}{1.202cm}}{\pgfqpoint{1.244cm}{1.202cm}}
\pgfpathcurveto{\pgfqpoint{1.207cm}{1.202cm}}{\pgfqpoint{1.173cm}{1.187cm}}{\pgfqpoint{1.147cm}{1.162cm}}
\pgfpathcurveto{\pgfqpoint{1.121cm}{1.136cm}}{\pgfqpoint{1.107cm}{1.101cm}}{\pgfqpoint{1.107cm}{1.065cm}}
\pgfpathcurveto{\pgfqpoint{1.107cm}{1.029cm}}{\pgfqpoint{1.121cm}{0.994cm}}{\pgfqpoint{1.147cm}{0.968cm}}
\pgfpathcurveto{\pgfqpoint{1.173cm}{0.942cm}}{\pgfqpoint{1.207cm}{0.928cm}}{\pgfqpoint{1.244cm}{0.928cm}}
\pgfpathcurveto{\pgfqpoint{1.28cm}{0.928cm}}{\pgfqpoint{1.315cm}{0.942cm}}{\pgfqpoint{1.34cm}{0.968cm}}
\pgfpathcurveto{\pgfqpoint{1.366cm}{0.994cm}}{\pgfqpoint{1.38cm}{1.029cm}}{\pgfqpoint{1.38cm}{1.065cm}}
\pgfusepath{fill}
\begin{pgfscope}
\pgfsetdash{}{0cm}
\pgfsetlinewidth{0.818mm}
\pgfsetmiterlimit{4.0}
\pgfpathmoveto{\pgfqpoint{1.383cm}{0.178cm}}
\pgfpathcurveto{\pgfqpoint{1.383cm}{0.214cm}}{\pgfqpoint{1.369cm}{0.249cm}}{\pgfqpoint{1.343cm}{0.275cm}}
\pgfpathcurveto{\pgfqpoint{1.317cm}{0.3cm}}{\pgfqpoint{1.283cm}{0.315cm}}{\pgfqpoint{1.246cm}{0.315cm}}
\pgfpathcurveto{\pgfqpoint{1.21cm}{0.315cm}}{\pgfqpoint{1.175cm}{0.3cm}}{\pgfqpoint{1.15cm}{0.275cm}}
\pgfpathcurveto{\pgfqpoint{1.124cm}{0.249cm}}{\pgfqpoint{1.11cm}{0.214cm}}{\pgfqpoint{1.11cm}{0.178cm}}
\pgfpathcurveto{\pgfqpoint{1.11cm}{0.141cm}}{\pgfqpoint{1.124cm}{0.107cm}}{\pgfqpoint{1.15cm}{0.081cm}}
\pgfpathcurveto{\pgfqpoint{1.175cm}{0.055cm}}{\pgfqpoint{1.21cm}{0.041cm}}{\pgfqpoint{1.246cm}{0.041cm}}
\pgfpathcurveto{\pgfqpoint{1.283cm}{0.041cm}}{\pgfqpoint{1.317cm}{0.055cm}}{\pgfqpoint{1.343cm}{0.081cm}}
\pgfpathcurveto{\pgfqpoint{1.369cm}{0.107cm}}{\pgfqpoint{1.383cm}{0.141cm}}{\pgfqpoint{1.383cm}{0.178cm}}
\pgfusepath{stroke}
\end{pgfscope}
\end{pgfscope}
\end{pgfscope}
\end{pgfscope}
\end{tikzpicture}}}_{M}
\end{equation}
given by the formulas above, such that if $\tau$ denotes one of the distributions in \eqref{eq:r45} and $\tau_{M}$ is the associated periodic version from \eqref{eq:r45M}, then $\tau\in \CC^{\alpha_{\tau}}(\rho^{\sigma})$ and $\tau_{M}\in \CC^{\alpha_{\tau}}(\mathbb{T}^{5}_{M})$ for $\alpha_{\tau}$  given by Table \ref{t:reg} and every $\kappa,\sigma>0$. Moreover,
$\tau_{M }\to \tau$ in $\CC^{\alpha_{\tau}}(\rho^{\sigma})$ a.s. as $M\to\infty$.
 \end{theorem}
 
 \begin{proof}
 Apart form the higher complexity of the terms involved in the $d=5$ case, the analysis proceeds like in Theorem~\ref{thm:renorm}. The various stochastic objects can be written as multiple iterated Wiener integrals and renormalizations accounts for cancellations of certain terms in the associated kernels. In the periodic and parabolic setting this analysis has already been performed several times with small variations, for example in~\cite{CC},~\cite{mourrat_construction_2016} and more recently in~\cite{furlan_weak_2017} and in~\cite{gubinelli_kpz_2017} for the KPZ equation. Estimation in weighted Besov spaces and convergence of the periodic to the non-periodic versions proceed like in the $\R^4$ case. 
 \end{proof}

\subsection{Space-time white noise}
\label{ssec:renorm-par}

If $\xi$ is a space white noise on $\R\times\R^{d}$, i.e. a family of centered Gaussian random variables $\{\xi(h);\,h\in L^{2}(\R\times\R^{d}) \}$ such that
$$
\mathbb{E}[\xi(h)^{2}]=\|h\|^{2}_{L^{2}},
$$
then we may define its periodization $\xi_{M}$ on $\mathbb{T}^{d}_{M}=\left[-\frac{M}{2},\frac{M}{2}\right]^{d}$  by
$$
\xi_{M}(h):=\xi(h_{M}),\qquad\text{where }\ h_{M}(t,x)=\mathbf{1}_{[-\frac{M}{2},\frac{M}{2}]^{d}}(x)\sum_{y\in M\mathbb{Z}^{d}}h(t,x+y).
$$
Our construction of solutions to the parabolic $\Phi^{4}$ model in Section \ref{sec:42} and Section \ref{sec:43} relies on  a smooth and space periodic approximation $\xi_{\varepsilon}$ of the driving space-time white noise $\xi$, defined on the torus of size $M=\frac 1\varepsilon$. To be more precise, let $\xi_{\varepsilon}$ be a periodic version of a space-time mollification of $\xi$ defined on $\R\times\mathbb{T}^{2}_{1/\varepsilon}$ and let $X,X_{\varepsilon}$ be stationary solutions to
\begin{equation*}
 \LL X = \xi, \qquad  \LL X_{\varepsilon} = \xi_{\varepsilon},
\end{equation*}
 and 
$$
 \llbracket X^2 \rrbracket  \assign \lim_{\varepsilon\to 0} X_{\varepsilon}^2 - a_{\varepsilon}, \qquad   \llbracket X^3 \rrbracket\assign\lim_{\varepsilon\to 0} X_{\varepsilon}^3 - 3 a_{\varepsilon} X_{\varepsilon},
 $$
 where again we can take $a_{\varepsilon}= \mathbbm{E}[X_{\varepsilon}^2(0,0)]$ is a diverging constant and the limits are understood in a suitable Besov space a.s. More precisely, the following result holds.

 \begin{theorem}\label{thm:renorm42}
Let $d=2$. Let $\rho(t,x)=\langle (t,x)\rangle^{-\nu}$ for some $\nu>0$. There exists a sequence of diverging constants $(a_{\varepsilon})_{\varepsilon\in (0,1)}$ and random distributions $X,\llbracket X^{2}\rrbracket,\llbracket X^{3}\rrbracket$  such that for all $\kappa,\sigma>0$  it holds
$$
\|X\|_{C\CC^{-\kappa}(\rho^{\sigma})},\|\llbracket X^{2}\rrbracket\|_{C\CC^{-\kappa}(\rho^{\sigma})}
,\|\llbracket X^{3}\rrbracket\|_{C\CC^{-\kappa}(\rho^{\sigma})}\lesssim 1,
$$
and
$$
 \llbracket X^2 \rrbracket  = \lim_{\varepsilon\to 0} X_{\varepsilon}^2 - a_{\varepsilon}, \qquad   \llbracket X^3 \rrbracket=\lim_{\varepsilon\to 0} X_{\varepsilon}^3 - 3 a_{\varepsilon} X_{\varepsilon},
 $$
where the limit is understood in $C\CC^{-\kappa}(\rho^{\sigma})$ a.s. as $\varepsilon\to\infty$.
\end{theorem}

 \begin{proof} The proof proceeds like in Theorem~\ref{thm:renorm} in the proof of which we also made reference to the relevant literature. We would just like to comment on how to obtain existence for all times within the claimed weighted space. Let $Y$ be one of the random fields considered in the theorem and $Y_\varepsilon$ the corresponding approximation. By standard estimates one obtains bounds of the form
 $$
 \mathbbm{E}[\|\rho^\sigma(t,\cdot) Y(t,\cdot)-\rho^\sigma(s,\cdot) Y(s,\cdot) \|_{\CC^{-\kappa}}^p]\lesssim |t-s|^{\delta p} \langle s \rangle^{-\beta p},
 $$ 
 for some small $\delta,\beta>0$ and large $p$, uniformly for $0\le s \le t \le s +1$.
 Standard Kolomogorov criterion can be applied to obtain that 
 $$
 \mathbbm{E}[\|\rho^\sigma Y \|_{C^{\delta'}([L,L+1];\CC^{-\kappa})}^p]\lesssim L^{-\beta p},
 $$ 
 for all $L\ge 1$ and some $\delta'>0$. Finally if $p$ is large enough this shows that the random variable
 $
 \sum_{L=0}^\infty \|\rho^\sigma Y \|_{C^{\delta'}([L,L+1];\CC^{-\kappa})}^p
 $ 
 has finite expectation. Finally a simple gluing argument implies that the random variable $\|\rho^\sigma Y \|_{C^{\delta'}(\R_+;\CC^{-\kappa})}$ has also finite $L^p$ moments. Weighted space convergence of the approximation $Y_\varepsilon$ to $Y$ can be handled similarly since we can establish that
 $$
\sup_L L^{\beta p}\mathbbm{E}[\|\rho^\sigma Y_\varepsilon - \rho^\sigma Y \|_{C^{\delta'}([L,L+1];\CC^{-\kappa})}^p]\lesssim  o_\varepsilon(1),
 $$ 
from which we obtain easily the convergence in the weighted norm $C^{\delta'}\CC^{-\kappa}(\rho^\sigma)$ as $\varepsilon \to 0$. 
 \end{proof}

Similarly to the elliptic 5 dimensional case, we define
\[ \LL X^{\!\resizebox{0.6em}{!}{
\begin{tikzpicture}
\pgfpathmoveto{\pgfqpoint{0cm}{-0.035cm}}
\pgfpathlineto{\pgfqpoint{1.376cm}{-0.035cm}}
\pgfpathlineto{\pgfqpoint{1.376cm}{1.552cm}}
\pgfpathlineto{\pgfqpoint{0cm}{1.552cm}}
\pgfpathclose
\pgfusepath{clip}
\begin{pgfscope}
\begin{pgfscope}
\pgfpathmoveto{\pgfqpoint{0cm}{-0.035cm}}
\pgfpathlineto{\pgfqpoint{1.376cm}{-0.035cm}}
\pgfpathlineto{\pgfqpoint{1.376cm}{1.552cm}}
\pgfpathlineto{\pgfqpoint{0cm}{1.552cm}}
\pgfpathclose
\pgfusepath{clip}
\begin{pgfscope}
\begin{pgfscope}
\pgfsetdash{}{0cm}
\pgfsetlinewidth{0.818mm}
\pgfsetroundcap
\pgfsetroundjoin
\pgfsetmiterlimit{7.0}
\definecolor{eps2pgf_color}{gray}{0}\pgfsetstrokecolor{eps2pgf_color}\pgfsetfillcolor{eps2pgf_color}
\pgfpathmoveto{\pgfqpoint{0.117cm}{1.421cm}}
\pgfpathlineto{\pgfqpoint{0.682cm}{0.671cm}}
\pgfpathlineto{\pgfqpoint{1.246cm}{1.421cm}}
\pgfusepath{stroke}
\end{pgfscope}
\definecolor{eps2pgf_color}{gray}{0}\pgfsetstrokecolor{eps2pgf_color}\pgfsetfillcolor{eps2pgf_color}
\pgfpathmoveto{\pgfqpoint{0.273cm}{1.395cm}}
\pgfpathcurveto{\pgfqpoint{0.273cm}{1.432cm}}{\pgfqpoint{0.259cm}{1.467cm}}{\pgfqpoint{0.233cm}{1.492cm}}
\pgfpathcurveto{\pgfqpoint{0.207cm}{1.518cm}}{\pgfqpoint{0.173cm}{1.532cm}}{\pgfqpoint{0.137cm}{1.532cm}}
\pgfpathcurveto{\pgfqpoint{0.1cm}{1.532cm}}{\pgfqpoint{0.066cm}{1.518cm}}{\pgfqpoint{0.04cm}{1.492cm}}
\pgfpathcurveto{\pgfqpoint{0.014cm}{1.467cm}}{\pgfqpoint{0cm}{1.432cm}}{\pgfqpoint{0cm}{1.395cm}}
\pgfpathcurveto{\pgfqpoint{0cm}{1.359cm}}{\pgfqpoint{0.014cm}{1.324cm}}{\pgfqpoint{0.04cm}{1.299cm}}
\pgfpathcurveto{\pgfqpoint{0.066cm}{1.273cm}}{\pgfqpoint{0.1cm}{1.258cm}}{\pgfqpoint{0.137cm}{1.258cm}}
\pgfpathcurveto{\pgfqpoint{0.173cm}{1.258cm}}{\pgfqpoint{0.207cm}{1.273cm}}{\pgfqpoint{0.233cm}{1.299cm}}
\pgfpathcurveto{\pgfqpoint{0.259cm}{1.324cm}}{\pgfqpoint{0.273cm}{1.359cm}}{\pgfqpoint{0.273cm}{1.395cm}}
\pgfusepath{fill}
\begin{pgfscope}
\pgfsetdash{}{0cm}
\pgfsetlinewidth{0.818mm}
\pgfsetmiterlimit{7.0}
\pgfpathmoveto{\pgfqpoint{0.682cm}{0.671cm}}
\pgfpathlineto{\pgfqpoint{0.679cm}{1.418cm}}
\pgfusepath{stroke}
\end{pgfscope}
\pgfpathmoveto{\pgfqpoint{0.815cm}{1.399cm}}
\pgfpathcurveto{\pgfqpoint{0.815cm}{1.435cm}}{\pgfqpoint{0.801cm}{1.47cm}}{\pgfqpoint{0.775cm}{1.496cm}}
\pgfpathcurveto{\pgfqpoint{0.75cm}{1.521cm}}{\pgfqpoint{0.715cm}{1.536cm}}{\pgfqpoint{0.679cm}{1.536cm}}
\pgfpathcurveto{\pgfqpoint{0.643cm}{1.536cm}}{\pgfqpoint{0.608cm}{1.521cm}}{\pgfqpoint{0.582cm}{1.496cm}}
\pgfpathcurveto{\pgfqpoint{0.557cm}{1.47cm}}{\pgfqpoint{0.542cm}{1.435cm}}{\pgfqpoint{0.542cm}{1.399cm}}
\pgfpathcurveto{\pgfqpoint{0.542cm}{1.363cm}}{\pgfqpoint{0.557cm}{1.328cm}}{\pgfqpoint{0.582cm}{1.302cm}}
\pgfpathcurveto{\pgfqpoint{0.608cm}{1.276cm}}{\pgfqpoint{0.643cm}{1.262cm}}{\pgfqpoint{0.679cm}{1.262cm}}
\pgfpathcurveto{\pgfqpoint{0.715cm}{1.262cm}}{\pgfqpoint{0.75cm}{1.276cm}}{\pgfqpoint{0.775cm}{1.302cm}}
\pgfpathcurveto{\pgfqpoint{0.801cm}{1.328cm}}{\pgfqpoint{0.815cm}{1.363cm}}{\pgfqpoint{0.815cm}{1.399cm}}
\pgfusepath{fill}
\pgfpathmoveto{\pgfqpoint{1.345cm}{1.371cm}}
\pgfpathcurveto{\pgfqpoint{1.345cm}{1.408cm}}{\pgfqpoint{1.331cm}{1.442cm}}{\pgfqpoint{1.305cm}{1.468cm}}
\pgfpathcurveto{\pgfqpoint{1.28cm}{1.494cm}}{\pgfqpoint{1.245cm}{1.508cm}}{\pgfqpoint{1.209cm}{1.508cm}}
\pgfpathcurveto{\pgfqpoint{1.172cm}{1.508cm}}{\pgfqpoint{1.138cm}{1.494cm}}{\pgfqpoint{1.112cm}{1.468cm}}
\pgfpathcurveto{\pgfqpoint{1.087cm}{1.442cm}}{\pgfqpoint{1.072cm}{1.408cm}}{\pgfqpoint{1.072cm}{1.371cm}}
\pgfpathcurveto{\pgfqpoint{1.072cm}{1.335cm}}{\pgfqpoint{1.087cm}{1.3cm}}{\pgfqpoint{1.112cm}{1.274cm}}
\pgfpathcurveto{\pgfqpoint{1.138cm}{1.249cm}}{\pgfqpoint{1.172cm}{1.234cm}}{\pgfqpoint{1.209cm}{1.234cm}}
\pgfpathcurveto{\pgfqpoint{1.245cm}{1.234cm}}{\pgfqpoint{1.28cm}{1.249cm}}{\pgfqpoint{1.305cm}{1.274cm}}
\pgfpathcurveto{\pgfqpoint{1.331cm}{1.3cm}}{\pgfqpoint{1.345cm}{1.335cm}}{\pgfqpoint{1.345cm}{1.371cm}}
\pgfusepath{fill}
\begin{pgfscope}
\pgfsetdash{}{0cm}
\pgfsetlinewidth{0.818mm}
\pgfsetroundcap
\pgfsetmiterlimit{4.0}
\pgfpathmoveto{\pgfqpoint{0.682cm}{0.671cm}}
\pgfpathlineto{\pgfqpoint{0.682cm}{0.042cm}}
\pgfusepath{stroke}
\end{pgfscope}
\end{pgfscope}
\end{pgfscope}
\end{pgfscope}
\end{tikzpicture}}} = \llbracket X^3 \rrbracket,\quadX^{\!\resizebox{0.6em}{!}{
\begin{tikzpicture}
\pgfpathmoveto{\pgfqpoint{0cm}{-0.035cm}}
\pgfpathlineto{\pgfqpoint{1.376cm}{-0.035cm}}
\pgfpathlineto{\pgfqpoint{1.376cm}{1.552cm}}
\pgfpathlineto{\pgfqpoint{0cm}{1.552cm}}
\pgfpathclose
\pgfusepath{clip}
\begin{pgfscope}
\begin{pgfscope}
\pgfpathmoveto{\pgfqpoint{0cm}{-0.035cm}}
\pgfpathlineto{\pgfqpoint{1.376cm}{-0.035cm}}
\pgfpathlineto{\pgfqpoint{1.376cm}{1.552cm}}
\pgfpathlineto{\pgfqpoint{0cm}{1.552cm}}
\pgfpathclose
\pgfusepath{clip}
\begin{pgfscope}
\begin{pgfscope}
\pgfsetdash{}{0cm}
\pgfsetlinewidth{0.818mm}
\pgfsetroundcap
\pgfsetroundjoin
\pgfsetmiterlimit{7.0}
\definecolor{eps2pgf_color}{gray}{0}\pgfsetstrokecolor{eps2pgf_color}\pgfsetfillcolor{eps2pgf_color}
\pgfpathmoveto{\pgfqpoint{0.117cm}{1.421cm}}
\pgfpathlineto{\pgfqpoint{0.682cm}{0.671cm}}
\pgfpathlineto{\pgfqpoint{1.246cm}{1.421cm}}
\pgfusepath{stroke}
\end{pgfscope}
\definecolor{eps2pgf_color}{gray}{0}\pgfsetstrokecolor{eps2pgf_color}\pgfsetfillcolor{eps2pgf_color}
\pgfpathmoveto{\pgfqpoint{0.273cm}{1.395cm}}
\pgfpathcurveto{\pgfqpoint{0.273cm}{1.432cm}}{\pgfqpoint{0.259cm}{1.467cm}}{\pgfqpoint{0.233cm}{1.492cm}}
\pgfpathcurveto{\pgfqpoint{0.207cm}{1.518cm}}{\pgfqpoint{0.173cm}{1.532cm}}{\pgfqpoint{0.137cm}{1.532cm}}
\pgfpathcurveto{\pgfqpoint{0.1cm}{1.532cm}}{\pgfqpoint{0.066cm}{1.518cm}}{\pgfqpoint{0.04cm}{1.492cm}}
\pgfpathcurveto{\pgfqpoint{0.014cm}{1.467cm}}{\pgfqpoint{0cm}{1.432cm}}{\pgfqpoint{0cm}{1.395cm}}
\pgfpathcurveto{\pgfqpoint{0cm}{1.359cm}}{\pgfqpoint{0.014cm}{1.324cm}}{\pgfqpoint{0.04cm}{1.299cm}}
\pgfpathcurveto{\pgfqpoint{0.066cm}{1.273cm}}{\pgfqpoint{0.1cm}{1.258cm}}{\pgfqpoint{0.137cm}{1.258cm}}
\pgfpathcurveto{\pgfqpoint{0.173cm}{1.258cm}}{\pgfqpoint{0.207cm}{1.273cm}}{\pgfqpoint{0.233cm}{1.299cm}}
\pgfpathcurveto{\pgfqpoint{0.259cm}{1.324cm}}{\pgfqpoint{0.273cm}{1.359cm}}{\pgfqpoint{0.273cm}{1.395cm}}
\pgfusepath{fill}
\begin{pgfscope}
\pgfsetdash{}{0cm}
\pgfsetlinewidth{0.818mm}
\pgfsetmiterlimit{7.0}
\pgfpathmoveto{\pgfqpoint{0.682cm}{0.671cm}}
\pgfpathlineto{\pgfqpoint{0.679cm}{1.418cm}}
\pgfusepath{stroke}
\end{pgfscope}
\pgfpathmoveto{\pgfqpoint{0.815cm}{1.399cm}}
\pgfpathcurveto{\pgfqpoint{0.815cm}{1.435cm}}{\pgfqpoint{0.801cm}{1.47cm}}{\pgfqpoint{0.775cm}{1.496cm}}
\pgfpathcurveto{\pgfqpoint{0.75cm}{1.521cm}}{\pgfqpoint{0.715cm}{1.536cm}}{\pgfqpoint{0.679cm}{1.536cm}}
\pgfpathcurveto{\pgfqpoint{0.643cm}{1.536cm}}{\pgfqpoint{0.608cm}{1.521cm}}{\pgfqpoint{0.582cm}{1.496cm}}
\pgfpathcurveto{\pgfqpoint{0.557cm}{1.47cm}}{\pgfqpoint{0.542cm}{1.435cm}}{\pgfqpoint{0.542cm}{1.399cm}}
\pgfpathcurveto{\pgfqpoint{0.542cm}{1.363cm}}{\pgfqpoint{0.557cm}{1.328cm}}{\pgfqpoint{0.582cm}{1.302cm}}
\pgfpathcurveto{\pgfqpoint{0.608cm}{1.276cm}}{\pgfqpoint{0.643cm}{1.262cm}}{\pgfqpoint{0.679cm}{1.262cm}}
\pgfpathcurveto{\pgfqpoint{0.715cm}{1.262cm}}{\pgfqpoint{0.75cm}{1.276cm}}{\pgfqpoint{0.775cm}{1.302cm}}
\pgfpathcurveto{\pgfqpoint{0.801cm}{1.328cm}}{\pgfqpoint{0.815cm}{1.363cm}}{\pgfqpoint{0.815cm}{1.399cm}}
\pgfusepath{fill}
\pgfpathmoveto{\pgfqpoint{1.345cm}{1.371cm}}
\pgfpathcurveto{\pgfqpoint{1.345cm}{1.408cm}}{\pgfqpoint{1.331cm}{1.442cm}}{\pgfqpoint{1.305cm}{1.468cm}}
\pgfpathcurveto{\pgfqpoint{1.28cm}{1.494cm}}{\pgfqpoint{1.245cm}{1.508cm}}{\pgfqpoint{1.209cm}{1.508cm}}
\pgfpathcurveto{\pgfqpoint{1.172cm}{1.508cm}}{\pgfqpoint{1.138cm}{1.494cm}}{\pgfqpoint{1.112cm}{1.468cm}}
\pgfpathcurveto{\pgfqpoint{1.087cm}{1.442cm}}{\pgfqpoint{1.072cm}{1.408cm}}{\pgfqpoint{1.072cm}{1.371cm}}
\pgfpathcurveto{\pgfqpoint{1.072cm}{1.335cm}}{\pgfqpoint{1.087cm}{1.3cm}}{\pgfqpoint{1.112cm}{1.274cm}}
\pgfpathcurveto{\pgfqpoint{1.138cm}{1.249cm}}{\pgfqpoint{1.172cm}{1.234cm}}{\pgfqpoint{1.209cm}{1.234cm}}
\pgfpathcurveto{\pgfqpoint{1.245cm}{1.234cm}}{\pgfqpoint{1.28cm}{1.249cm}}{\pgfqpoint{1.305cm}{1.274cm}}
\pgfpathcurveto{\pgfqpoint{1.331cm}{1.3cm}}{\pgfqpoint{1.345cm}{1.335cm}}{\pgfqpoint{1.345cm}{1.371cm}}
\pgfusepath{fill}
\begin{pgfscope}
\pgfsetdash{}{0cm}
\pgfsetlinewidth{0.818mm}
\pgfsetroundcap
\pgfsetmiterlimit{4.0}
\pgfpathmoveto{\pgfqpoint{0.682cm}{0.671cm}}
\pgfpathlineto{\pgfqpoint{0.682cm}{0.042cm}}
\pgfusepath{stroke}
\end{pgfscope}
\end{pgfscope}
\end{pgfscope}
\end{pgfscope}
\end{tikzpicture}}}(0)=0,   \qquad \LL
   X^{\!\resizebox{0.6em}{!}{
\begin{tikzpicture}
\pgfpathmoveto{\pgfqpoint{0cm}{0cm}}
\pgfpathlineto{\pgfqpoint{1.376cm}{0cm}}
\pgfpathlineto{\pgfqpoint{1.376cm}{1.588cm}}
\pgfpathlineto{\pgfqpoint{0cm}{1.588cm}}
\pgfpathclose
\pgfusepath{clip}
\begin{pgfscope}
\begin{pgfscope}
\pgfpathmoveto{\pgfqpoint{0cm}{0cm}}
\pgfpathlineto{\pgfqpoint{1.376cm}{0cm}}
\pgfpathlineto{\pgfqpoint{1.376cm}{1.588cm}}
\pgfpathlineto{\pgfqpoint{0cm}{1.588cm}}
\pgfpathclose
\pgfusepath{clip}
\begin{pgfscope}
\begin{pgfscope}
\definecolor{eps2pgf_color}{gray}{0.976471}\pgfsetstrokecolor{eps2pgf_color}\pgfsetfillcolor{eps2pgf_color}
\pgfpathmoveto{\pgfqpoint{0cm}{0cm}}
\pgfpathlineto{\pgfqpoint{1.376cm}{0cm}}
\pgfpathlineto{\pgfqpoint{1.376cm}{1.588cm}}
\pgfpathlineto{\pgfqpoint{0cm}{1.588cm}}
\pgfpathclose
\pgfusepath{fill}
\end{pgfscope}
\begin{pgfscope}
\pgfsetdash{}{0cm}
\pgfsetlinewidth{0.818mm}
\pgfsetroundcap
\pgfsetroundjoin
\pgfsetmiterlimit{7.0}
\definecolor{eps2pgf_color}{gray}{0}\pgfsetstrokecolor{eps2pgf_color}\pgfsetfillcolor{eps2pgf_color}
\pgfpathmoveto{\pgfqpoint{0.117cm}{1.476cm}}
\pgfpathlineto{\pgfqpoint{0.682cm}{0.726cm}}
\pgfpathlineto{\pgfqpoint{1.246cm}{1.476cm}}
\pgfusepath{stroke}
\end{pgfscope}
\definecolor{eps2pgf_color}{gray}{0}\pgfsetstrokecolor{eps2pgf_color}\pgfsetfillcolor{eps2pgf_color}
\pgfpathmoveto{\pgfqpoint{0.273cm}{1.451cm}}
\pgfpathcurveto{\pgfqpoint{0.273cm}{1.487cm}}{\pgfqpoint{0.259cm}{1.522cm}}{\pgfqpoint{0.233cm}{1.547cm}}
\pgfpathcurveto{\pgfqpoint{0.207cm}{1.573cm}}{\pgfqpoint{0.173cm}{1.588cm}}{\pgfqpoint{0.137cm}{1.588cm}}
\pgfpathcurveto{\pgfqpoint{0.1cm}{1.588cm}}{\pgfqpoint{0.066cm}{1.573cm}}{\pgfqpoint{0.04cm}{1.547cm}}
\pgfpathcurveto{\pgfqpoint{0.014cm}{1.522cm}}{\pgfqpoint{0cm}{1.487cm}}{\pgfqpoint{0cm}{1.451cm}}
\pgfpathcurveto{\pgfqpoint{0cm}{1.414cm}}{\pgfqpoint{0.014cm}{1.379cm}}{\pgfqpoint{0.04cm}{1.354cm}}
\pgfpathcurveto{\pgfqpoint{0.066cm}{1.328cm}}{\pgfqpoint{0.1cm}{1.314cm}}{\pgfqpoint{0.137cm}{1.314cm}}
\pgfpathcurveto{\pgfqpoint{0.173cm}{1.314cm}}{\pgfqpoint{0.207cm}{1.328cm}}{\pgfqpoint{0.233cm}{1.354cm}}
\pgfpathcurveto{\pgfqpoint{0.259cm}{1.379cm}}{\pgfqpoint{0.273cm}{1.414cm}}{\pgfqpoint{0.273cm}{1.451cm}}
\pgfusepath{fill}
\pgfpathmoveto{\pgfqpoint{1.345cm}{1.426cm}}
\pgfpathcurveto{\pgfqpoint{1.345cm}{1.463cm}}{\pgfqpoint{1.331cm}{1.497cm}}{\pgfqpoint{1.305cm}{1.523cm}}
\pgfpathcurveto{\pgfqpoint{1.28cm}{1.549cm}}{\pgfqpoint{1.245cm}{1.563cm}}{\pgfqpoint{1.209cm}{1.563cm}}
\pgfpathcurveto{\pgfqpoint{1.172cm}{1.563cm}}{\pgfqpoint{1.138cm}{1.549cm}}{\pgfqpoint{1.112cm}{1.523cm}}
\pgfpathcurveto{\pgfqpoint{1.087cm}{1.497cm}}{\pgfqpoint{1.072cm}{1.463cm}}{\pgfqpoint{1.072cm}{1.426cm}}
\pgfpathcurveto{\pgfqpoint{1.072cm}{1.39cm}}{\pgfqpoint{1.087cm}{1.355cm}}{\pgfqpoint{1.112cm}{1.329cm}}
\pgfpathcurveto{\pgfqpoint{1.138cm}{1.304cm}}{\pgfqpoint{1.172cm}{1.289cm}}{\pgfqpoint{1.209cm}{1.289cm}}
\pgfpathcurveto{\pgfqpoint{1.245cm}{1.289cm}}{\pgfqpoint{1.28cm}{1.304cm}}{\pgfqpoint{1.305cm}{1.329cm}}
\pgfpathcurveto{\pgfqpoint{1.331cm}{1.355cm}}{\pgfqpoint{1.345cm}{1.39cm}}{\pgfqpoint{1.345cm}{1.426cm}}
\pgfusepath{fill}
\begin{pgfscope}
\pgfsetdash{}{0cm}
\pgfsetlinewidth{0.818mm}
\pgfsetroundcap
\pgfsetmiterlimit{4.0}
\pgfpathmoveto{\pgfqpoint{0.682cm}{0.726cm}}
\pgfpathlineto{\pgfqpoint{0.682cm}{0.097cm}}
\pgfusepath{stroke}
\end{pgfscope}
\end{pgfscope}
\end{pgfscope}
\end{pgfscope}
\end{tikzpicture}}} = \llbracket X^2 \rrbracket, \quad X^{\!\resizebox{0.6em}{!}{
\begin{tikzpicture}
\pgfpathmoveto{\pgfqpoint{0cm}{0cm}}
\pgfpathlineto{\pgfqpoint{1.376cm}{0cm}}
\pgfpathlineto{\pgfqpoint{1.376cm}{1.588cm}}
\pgfpathlineto{\pgfqpoint{0cm}{1.588cm}}
\pgfpathclose
\pgfusepath{clip}
\begin{pgfscope}
\begin{pgfscope}
\pgfpathmoveto{\pgfqpoint{0cm}{0cm}}
\pgfpathlineto{\pgfqpoint{1.376cm}{0cm}}
\pgfpathlineto{\pgfqpoint{1.376cm}{1.588cm}}
\pgfpathlineto{\pgfqpoint{0cm}{1.588cm}}
\pgfpathclose
\pgfusepath{clip}
\begin{pgfscope}
\begin{pgfscope}
\definecolor{eps2pgf_color}{gray}{0.976471}\pgfsetstrokecolor{eps2pgf_color}\pgfsetfillcolor{eps2pgf_color}
\pgfpathmoveto{\pgfqpoint{0cm}{0cm}}
\pgfpathlineto{\pgfqpoint{1.376cm}{0cm}}
\pgfpathlineto{\pgfqpoint{1.376cm}{1.588cm}}
\pgfpathlineto{\pgfqpoint{0cm}{1.588cm}}
\pgfpathclose
\pgfusepath{fill}
\end{pgfscope}
\begin{pgfscope}
\pgfsetdash{}{0cm}
\pgfsetlinewidth{0.818mm}
\pgfsetroundcap
\pgfsetroundjoin
\pgfsetmiterlimit{7.0}
\definecolor{eps2pgf_color}{gray}{0}\pgfsetstrokecolor{eps2pgf_color}\pgfsetfillcolor{eps2pgf_color}
\pgfpathmoveto{\pgfqpoint{0.117cm}{1.476cm}}
\pgfpathlineto{\pgfqpoint{0.682cm}{0.726cm}}
\pgfpathlineto{\pgfqpoint{1.246cm}{1.476cm}}
\pgfusepath{stroke}
\end{pgfscope}
\definecolor{eps2pgf_color}{gray}{0}\pgfsetstrokecolor{eps2pgf_color}\pgfsetfillcolor{eps2pgf_color}
\pgfpathmoveto{\pgfqpoint{0.273cm}{1.451cm}}
\pgfpathcurveto{\pgfqpoint{0.273cm}{1.487cm}}{\pgfqpoint{0.259cm}{1.522cm}}{\pgfqpoint{0.233cm}{1.547cm}}
\pgfpathcurveto{\pgfqpoint{0.207cm}{1.573cm}}{\pgfqpoint{0.173cm}{1.588cm}}{\pgfqpoint{0.137cm}{1.588cm}}
\pgfpathcurveto{\pgfqpoint{0.1cm}{1.588cm}}{\pgfqpoint{0.066cm}{1.573cm}}{\pgfqpoint{0.04cm}{1.547cm}}
\pgfpathcurveto{\pgfqpoint{0.014cm}{1.522cm}}{\pgfqpoint{0cm}{1.487cm}}{\pgfqpoint{0cm}{1.451cm}}
\pgfpathcurveto{\pgfqpoint{0cm}{1.414cm}}{\pgfqpoint{0.014cm}{1.379cm}}{\pgfqpoint{0.04cm}{1.354cm}}
\pgfpathcurveto{\pgfqpoint{0.066cm}{1.328cm}}{\pgfqpoint{0.1cm}{1.314cm}}{\pgfqpoint{0.137cm}{1.314cm}}
\pgfpathcurveto{\pgfqpoint{0.173cm}{1.314cm}}{\pgfqpoint{0.207cm}{1.328cm}}{\pgfqpoint{0.233cm}{1.354cm}}
\pgfpathcurveto{\pgfqpoint{0.259cm}{1.379cm}}{\pgfqpoint{0.273cm}{1.414cm}}{\pgfqpoint{0.273cm}{1.451cm}}
\pgfusepath{fill}
\pgfpathmoveto{\pgfqpoint{1.345cm}{1.426cm}}
\pgfpathcurveto{\pgfqpoint{1.345cm}{1.463cm}}{\pgfqpoint{1.331cm}{1.497cm}}{\pgfqpoint{1.305cm}{1.523cm}}
\pgfpathcurveto{\pgfqpoint{1.28cm}{1.549cm}}{\pgfqpoint{1.245cm}{1.563cm}}{\pgfqpoint{1.209cm}{1.563cm}}
\pgfpathcurveto{\pgfqpoint{1.172cm}{1.563cm}}{\pgfqpoint{1.138cm}{1.549cm}}{\pgfqpoint{1.112cm}{1.523cm}}
\pgfpathcurveto{\pgfqpoint{1.087cm}{1.497cm}}{\pgfqpoint{1.072cm}{1.463cm}}{\pgfqpoint{1.072cm}{1.426cm}}
\pgfpathcurveto{\pgfqpoint{1.072cm}{1.39cm}}{\pgfqpoint{1.087cm}{1.355cm}}{\pgfqpoint{1.112cm}{1.329cm}}
\pgfpathcurveto{\pgfqpoint{1.138cm}{1.304cm}}{\pgfqpoint{1.172cm}{1.289cm}}{\pgfqpoint{1.209cm}{1.289cm}}
\pgfpathcurveto{\pgfqpoint{1.245cm}{1.289cm}}{\pgfqpoint{1.28cm}{1.304cm}}{\pgfqpoint{1.305cm}{1.329cm}}
\pgfpathcurveto{\pgfqpoint{1.331cm}{1.355cm}}{\pgfqpoint{1.345cm}{1.39cm}}{\pgfqpoint{1.345cm}{1.426cm}}
\pgfusepath{fill}
\begin{pgfscope}
\pgfsetdash{}{0cm}
\pgfsetlinewidth{0.818mm}
\pgfsetroundcap
\pgfsetmiterlimit{4.0}
\pgfpathmoveto{\pgfqpoint{0.682cm}{0.726cm}}
\pgfpathlineto{\pgfqpoint{0.682cm}{0.097cm}}
\pgfusepath{stroke}
\end{pgfscope}
\end{pgfscope}
\end{pgfscope}
\end{pgfscope}
\end{tikzpicture}}}(0)=0,\]
      \[ \LL X^{\!\resizebox{0.6em}{!}{
\begin{tikzpicture}
\pgfpathmoveto{\pgfqpoint{0cm}{-0.035cm}}
\pgfpathlineto{\pgfqpoint{1.376cm}{-0.035cm}}
\pgfpathlineto{\pgfqpoint{1.376cm}{1.552cm}}
\pgfpathlineto{\pgfqpoint{0cm}{1.552cm}}
\pgfpathclose
\pgfusepath{clip}
\begin{pgfscope}
\begin{pgfscope}
\pgfpathmoveto{\pgfqpoint{0cm}{-0.035cm}}
\pgfpathlineto{\pgfqpoint{1.376cm}{-0.035cm}}
\pgfpathlineto{\pgfqpoint{1.376cm}{1.552cm}}
\pgfpathlineto{\pgfqpoint{0cm}{1.552cm}}
\pgfpathclose
\pgfusepath{clip}
\begin{pgfscope}
\begin{pgfscope}
\pgfsetdash{}{0cm}
\pgfsetlinewidth{0.818mm}
\pgfsetroundcap
\pgfsetroundjoin
\pgfsetmiterlimit{7.0}
\definecolor{eps2pgf_color}{gray}{0}\pgfsetstrokecolor{eps2pgf_color}\pgfsetfillcolor{eps2pgf_color}
\pgfpathmoveto{\pgfqpoint{0.117cm}{1.421cm}}
\pgfpathlineto{\pgfqpoint{0.682cm}{0.671cm}}
\pgfpathlineto{\pgfqpoint{1.246cm}{1.421cm}}
\pgfusepath{stroke}
\end{pgfscope}
\definecolor{eps2pgf_color}{gray}{0}\pgfsetstrokecolor{eps2pgf_color}\pgfsetfillcolor{eps2pgf_color}
\pgfpathmoveto{\pgfqpoint{0.273cm}{1.395cm}}
\pgfpathcurveto{\pgfqpoint{0.273cm}{1.432cm}}{\pgfqpoint{0.259cm}{1.467cm}}{\pgfqpoint{0.233cm}{1.492cm}}
\pgfpathcurveto{\pgfqpoint{0.207cm}{1.518cm}}{\pgfqpoint{0.173cm}{1.532cm}}{\pgfqpoint{0.137cm}{1.532cm}}
\pgfpathcurveto{\pgfqpoint{0.1cm}{1.532cm}}{\pgfqpoint{0.066cm}{1.518cm}}{\pgfqpoint{0.04cm}{1.492cm}}
\pgfpathcurveto{\pgfqpoint{0.014cm}{1.467cm}}{\pgfqpoint{0cm}{1.432cm}}{\pgfqpoint{0cm}{1.395cm}}
\pgfpathcurveto{\pgfqpoint{0cm}{1.359cm}}{\pgfqpoint{0.014cm}{1.324cm}}{\pgfqpoint{0.04cm}{1.299cm}}
\pgfpathcurveto{\pgfqpoint{0.066cm}{1.273cm}}{\pgfqpoint{0.1cm}{1.258cm}}{\pgfqpoint{0.137cm}{1.258cm}}
\pgfpathcurveto{\pgfqpoint{0.173cm}{1.258cm}}{\pgfqpoint{0.207cm}{1.273cm}}{\pgfqpoint{0.233cm}{1.299cm}}
\pgfpathcurveto{\pgfqpoint{0.259cm}{1.324cm}}{\pgfqpoint{0.273cm}{1.359cm}}{\pgfqpoint{0.273cm}{1.395cm}}
\pgfusepath{fill}
\begin{pgfscope}
\pgfsetdash{}{0cm}
\pgfsetlinewidth{0.818mm}
\pgfsetmiterlimit{7.0}
\pgfpathmoveto{\pgfqpoint{0.682cm}{0.671cm}}
\pgfpathlineto{\pgfqpoint{0.679cm}{1.418cm}}
\pgfusepath{stroke}
\end{pgfscope}
\pgfpathmoveto{\pgfqpoint{0.815cm}{1.399cm}}
\pgfpathcurveto{\pgfqpoint{0.815cm}{1.435cm}}{\pgfqpoint{0.801cm}{1.47cm}}{\pgfqpoint{0.775cm}{1.496cm}}
\pgfpathcurveto{\pgfqpoint{0.75cm}{1.521cm}}{\pgfqpoint{0.715cm}{1.536cm}}{\pgfqpoint{0.679cm}{1.536cm}}
\pgfpathcurveto{\pgfqpoint{0.643cm}{1.536cm}}{\pgfqpoint{0.608cm}{1.521cm}}{\pgfqpoint{0.582cm}{1.496cm}}
\pgfpathcurveto{\pgfqpoint{0.557cm}{1.47cm}}{\pgfqpoint{0.542cm}{1.435cm}}{\pgfqpoint{0.542cm}{1.399cm}}
\pgfpathcurveto{\pgfqpoint{0.542cm}{1.363cm}}{\pgfqpoint{0.557cm}{1.328cm}}{\pgfqpoint{0.582cm}{1.302cm}}
\pgfpathcurveto{\pgfqpoint{0.608cm}{1.276cm}}{\pgfqpoint{0.643cm}{1.262cm}}{\pgfqpoint{0.679cm}{1.262cm}}
\pgfpathcurveto{\pgfqpoint{0.715cm}{1.262cm}}{\pgfqpoint{0.75cm}{1.276cm}}{\pgfqpoint{0.775cm}{1.302cm}}
\pgfpathcurveto{\pgfqpoint{0.801cm}{1.328cm}}{\pgfqpoint{0.815cm}{1.363cm}}{\pgfqpoint{0.815cm}{1.399cm}}
\pgfusepath{fill}
\pgfpathmoveto{\pgfqpoint{1.345cm}{1.371cm}}
\pgfpathcurveto{\pgfqpoint{1.345cm}{1.408cm}}{\pgfqpoint{1.331cm}{1.442cm}}{\pgfqpoint{1.305cm}{1.468cm}}
\pgfpathcurveto{\pgfqpoint{1.28cm}{1.494cm}}{\pgfqpoint{1.245cm}{1.508cm}}{\pgfqpoint{1.209cm}{1.508cm}}
\pgfpathcurveto{\pgfqpoint{1.172cm}{1.508cm}}{\pgfqpoint{1.138cm}{1.494cm}}{\pgfqpoint{1.112cm}{1.468cm}}
\pgfpathcurveto{\pgfqpoint{1.087cm}{1.442cm}}{\pgfqpoint{1.072cm}{1.408cm}}{\pgfqpoint{1.072cm}{1.371cm}}
\pgfpathcurveto{\pgfqpoint{1.072cm}{1.335cm}}{\pgfqpoint{1.087cm}{1.3cm}}{\pgfqpoint{1.112cm}{1.274cm}}
\pgfpathcurveto{\pgfqpoint{1.138cm}{1.249cm}}{\pgfqpoint{1.172cm}{1.234cm}}{\pgfqpoint{1.209cm}{1.234cm}}
\pgfpathcurveto{\pgfqpoint{1.245cm}{1.234cm}}{\pgfqpoint{1.28cm}{1.249cm}}{\pgfqpoint{1.305cm}{1.274cm}}
\pgfpathcurveto{\pgfqpoint{1.331cm}{1.3cm}}{\pgfqpoint{1.345cm}{1.335cm}}{\pgfqpoint{1.345cm}{1.371cm}}
\pgfusepath{fill}
\begin{pgfscope}
\pgfsetdash{}{0cm}
\pgfsetlinewidth{0.818mm}
\pgfsetroundcap
\pgfsetmiterlimit{4.0}
\pgfpathmoveto{\pgfqpoint{0.682cm}{0.671cm}}
\pgfpathlineto{\pgfqpoint{0.682cm}{0.042cm}}
\pgfusepath{stroke}
\end{pgfscope}
\end{pgfscope}
\end{pgfscope}
\end{pgfscope}
\end{tikzpicture}}}_{\varepsilon} = \llbracket X_{\varepsilon}^3 \rrbracket,\quad X^{\!\resizebox{0.6em}{!}{
\begin{tikzpicture}
\pgfpathmoveto{\pgfqpoint{0cm}{-0.035cm}}
\pgfpathlineto{\pgfqpoint{1.376cm}{-0.035cm}}
\pgfpathlineto{\pgfqpoint{1.376cm}{1.552cm}}
\pgfpathlineto{\pgfqpoint{0cm}{1.552cm}}
\pgfpathclose
\pgfusepath{clip}
\begin{pgfscope}
\begin{pgfscope}
\pgfpathmoveto{\pgfqpoint{0cm}{-0.035cm}}
\pgfpathlineto{\pgfqpoint{1.376cm}{-0.035cm}}
\pgfpathlineto{\pgfqpoint{1.376cm}{1.552cm}}
\pgfpathlineto{\pgfqpoint{0cm}{1.552cm}}
\pgfpathclose
\pgfusepath{clip}
\begin{pgfscope}
\begin{pgfscope}
\pgfsetdash{}{0cm}
\pgfsetlinewidth{0.818mm}
\pgfsetroundcap
\pgfsetroundjoin
\pgfsetmiterlimit{7.0}
\definecolor{eps2pgf_color}{gray}{0}\pgfsetstrokecolor{eps2pgf_color}\pgfsetfillcolor{eps2pgf_color}
\pgfpathmoveto{\pgfqpoint{0.117cm}{1.421cm}}
\pgfpathlineto{\pgfqpoint{0.682cm}{0.671cm}}
\pgfpathlineto{\pgfqpoint{1.246cm}{1.421cm}}
\pgfusepath{stroke}
\end{pgfscope}
\definecolor{eps2pgf_color}{gray}{0}\pgfsetstrokecolor{eps2pgf_color}\pgfsetfillcolor{eps2pgf_color}
\pgfpathmoveto{\pgfqpoint{0.273cm}{1.395cm}}
\pgfpathcurveto{\pgfqpoint{0.273cm}{1.432cm}}{\pgfqpoint{0.259cm}{1.467cm}}{\pgfqpoint{0.233cm}{1.492cm}}
\pgfpathcurveto{\pgfqpoint{0.207cm}{1.518cm}}{\pgfqpoint{0.173cm}{1.532cm}}{\pgfqpoint{0.137cm}{1.532cm}}
\pgfpathcurveto{\pgfqpoint{0.1cm}{1.532cm}}{\pgfqpoint{0.066cm}{1.518cm}}{\pgfqpoint{0.04cm}{1.492cm}}
\pgfpathcurveto{\pgfqpoint{0.014cm}{1.467cm}}{\pgfqpoint{0cm}{1.432cm}}{\pgfqpoint{0cm}{1.395cm}}
\pgfpathcurveto{\pgfqpoint{0cm}{1.359cm}}{\pgfqpoint{0.014cm}{1.324cm}}{\pgfqpoint{0.04cm}{1.299cm}}
\pgfpathcurveto{\pgfqpoint{0.066cm}{1.273cm}}{\pgfqpoint{0.1cm}{1.258cm}}{\pgfqpoint{0.137cm}{1.258cm}}
\pgfpathcurveto{\pgfqpoint{0.173cm}{1.258cm}}{\pgfqpoint{0.207cm}{1.273cm}}{\pgfqpoint{0.233cm}{1.299cm}}
\pgfpathcurveto{\pgfqpoint{0.259cm}{1.324cm}}{\pgfqpoint{0.273cm}{1.359cm}}{\pgfqpoint{0.273cm}{1.395cm}}
\pgfusepath{fill}
\begin{pgfscope}
\pgfsetdash{}{0cm}
\pgfsetlinewidth{0.818mm}
\pgfsetmiterlimit{7.0}
\pgfpathmoveto{\pgfqpoint{0.682cm}{0.671cm}}
\pgfpathlineto{\pgfqpoint{0.679cm}{1.418cm}}
\pgfusepath{stroke}
\end{pgfscope}
\pgfpathmoveto{\pgfqpoint{0.815cm}{1.399cm}}
\pgfpathcurveto{\pgfqpoint{0.815cm}{1.435cm}}{\pgfqpoint{0.801cm}{1.47cm}}{\pgfqpoint{0.775cm}{1.496cm}}
\pgfpathcurveto{\pgfqpoint{0.75cm}{1.521cm}}{\pgfqpoint{0.715cm}{1.536cm}}{\pgfqpoint{0.679cm}{1.536cm}}
\pgfpathcurveto{\pgfqpoint{0.643cm}{1.536cm}}{\pgfqpoint{0.608cm}{1.521cm}}{\pgfqpoint{0.582cm}{1.496cm}}
\pgfpathcurveto{\pgfqpoint{0.557cm}{1.47cm}}{\pgfqpoint{0.542cm}{1.435cm}}{\pgfqpoint{0.542cm}{1.399cm}}
\pgfpathcurveto{\pgfqpoint{0.542cm}{1.363cm}}{\pgfqpoint{0.557cm}{1.328cm}}{\pgfqpoint{0.582cm}{1.302cm}}
\pgfpathcurveto{\pgfqpoint{0.608cm}{1.276cm}}{\pgfqpoint{0.643cm}{1.262cm}}{\pgfqpoint{0.679cm}{1.262cm}}
\pgfpathcurveto{\pgfqpoint{0.715cm}{1.262cm}}{\pgfqpoint{0.75cm}{1.276cm}}{\pgfqpoint{0.775cm}{1.302cm}}
\pgfpathcurveto{\pgfqpoint{0.801cm}{1.328cm}}{\pgfqpoint{0.815cm}{1.363cm}}{\pgfqpoint{0.815cm}{1.399cm}}
\pgfusepath{fill}
\pgfpathmoveto{\pgfqpoint{1.345cm}{1.371cm}}
\pgfpathcurveto{\pgfqpoint{1.345cm}{1.408cm}}{\pgfqpoint{1.331cm}{1.442cm}}{\pgfqpoint{1.305cm}{1.468cm}}
\pgfpathcurveto{\pgfqpoint{1.28cm}{1.494cm}}{\pgfqpoint{1.245cm}{1.508cm}}{\pgfqpoint{1.209cm}{1.508cm}}
\pgfpathcurveto{\pgfqpoint{1.172cm}{1.508cm}}{\pgfqpoint{1.138cm}{1.494cm}}{\pgfqpoint{1.112cm}{1.468cm}}
\pgfpathcurveto{\pgfqpoint{1.087cm}{1.442cm}}{\pgfqpoint{1.072cm}{1.408cm}}{\pgfqpoint{1.072cm}{1.371cm}}
\pgfpathcurveto{\pgfqpoint{1.072cm}{1.335cm}}{\pgfqpoint{1.087cm}{1.3cm}}{\pgfqpoint{1.112cm}{1.274cm}}
\pgfpathcurveto{\pgfqpoint{1.138cm}{1.249cm}}{\pgfqpoint{1.172cm}{1.234cm}}{\pgfqpoint{1.209cm}{1.234cm}}
\pgfpathcurveto{\pgfqpoint{1.245cm}{1.234cm}}{\pgfqpoint{1.28cm}{1.249cm}}{\pgfqpoint{1.305cm}{1.274cm}}
\pgfpathcurveto{\pgfqpoint{1.331cm}{1.3cm}}{\pgfqpoint{1.345cm}{1.335cm}}{\pgfqpoint{1.345cm}{1.371cm}}
\pgfusepath{fill}
\begin{pgfscope}
\pgfsetdash{}{0cm}
\pgfsetlinewidth{0.818mm}
\pgfsetroundcap
\pgfsetmiterlimit{4.0}
\pgfpathmoveto{\pgfqpoint{0.682cm}{0.671cm}}
\pgfpathlineto{\pgfqpoint{0.682cm}{0.042cm}}
\pgfusepath{stroke}
\end{pgfscope}
\end{pgfscope}
\end{pgfscope}
\end{pgfscope}
\end{tikzpicture}}}_{\varepsilon}(0)=0, \qquad \LL
   X^{\!\resizebox{0.6em}{!}{
\begin{tikzpicture}
\pgfpathmoveto{\pgfqpoint{0cm}{0cm}}
\pgfpathlineto{\pgfqpoint{1.376cm}{0cm}}
\pgfpathlineto{\pgfqpoint{1.376cm}{1.588cm}}
\pgfpathlineto{\pgfqpoint{0cm}{1.588cm}}
\pgfpathclose
\pgfusepath{clip}
\begin{pgfscope}
\begin{pgfscope}
\pgfpathmoveto{\pgfqpoint{0cm}{0cm}}
\pgfpathlineto{\pgfqpoint{1.376cm}{0cm}}
\pgfpathlineto{\pgfqpoint{1.376cm}{1.588cm}}
\pgfpathlineto{\pgfqpoint{0cm}{1.588cm}}
\pgfpathclose
\pgfusepath{clip}
\begin{pgfscope}
\begin{pgfscope}
\definecolor{eps2pgf_color}{gray}{0.976471}\pgfsetstrokecolor{eps2pgf_color}\pgfsetfillcolor{eps2pgf_color}
\pgfpathmoveto{\pgfqpoint{0cm}{0cm}}
\pgfpathlineto{\pgfqpoint{1.376cm}{0cm}}
\pgfpathlineto{\pgfqpoint{1.376cm}{1.588cm}}
\pgfpathlineto{\pgfqpoint{0cm}{1.588cm}}
\pgfpathclose
\pgfusepath{fill}
\end{pgfscope}
\begin{pgfscope}
\pgfsetdash{}{0cm}
\pgfsetlinewidth{0.818mm}
\pgfsetroundcap
\pgfsetroundjoin
\pgfsetmiterlimit{7.0}
\definecolor{eps2pgf_color}{gray}{0}\pgfsetstrokecolor{eps2pgf_color}\pgfsetfillcolor{eps2pgf_color}
\pgfpathmoveto{\pgfqpoint{0.117cm}{1.476cm}}
\pgfpathlineto{\pgfqpoint{0.682cm}{0.726cm}}
\pgfpathlineto{\pgfqpoint{1.246cm}{1.476cm}}
\pgfusepath{stroke}
\end{pgfscope}
\definecolor{eps2pgf_color}{gray}{0}\pgfsetstrokecolor{eps2pgf_color}\pgfsetfillcolor{eps2pgf_color}
\pgfpathmoveto{\pgfqpoint{0.273cm}{1.451cm}}
\pgfpathcurveto{\pgfqpoint{0.273cm}{1.487cm}}{\pgfqpoint{0.259cm}{1.522cm}}{\pgfqpoint{0.233cm}{1.547cm}}
\pgfpathcurveto{\pgfqpoint{0.207cm}{1.573cm}}{\pgfqpoint{0.173cm}{1.588cm}}{\pgfqpoint{0.137cm}{1.588cm}}
\pgfpathcurveto{\pgfqpoint{0.1cm}{1.588cm}}{\pgfqpoint{0.066cm}{1.573cm}}{\pgfqpoint{0.04cm}{1.547cm}}
\pgfpathcurveto{\pgfqpoint{0.014cm}{1.522cm}}{\pgfqpoint{0cm}{1.487cm}}{\pgfqpoint{0cm}{1.451cm}}
\pgfpathcurveto{\pgfqpoint{0cm}{1.414cm}}{\pgfqpoint{0.014cm}{1.379cm}}{\pgfqpoint{0.04cm}{1.354cm}}
\pgfpathcurveto{\pgfqpoint{0.066cm}{1.328cm}}{\pgfqpoint{0.1cm}{1.314cm}}{\pgfqpoint{0.137cm}{1.314cm}}
\pgfpathcurveto{\pgfqpoint{0.173cm}{1.314cm}}{\pgfqpoint{0.207cm}{1.328cm}}{\pgfqpoint{0.233cm}{1.354cm}}
\pgfpathcurveto{\pgfqpoint{0.259cm}{1.379cm}}{\pgfqpoint{0.273cm}{1.414cm}}{\pgfqpoint{0.273cm}{1.451cm}}
\pgfusepath{fill}
\pgfpathmoveto{\pgfqpoint{1.345cm}{1.426cm}}
\pgfpathcurveto{\pgfqpoint{1.345cm}{1.463cm}}{\pgfqpoint{1.331cm}{1.497cm}}{\pgfqpoint{1.305cm}{1.523cm}}
\pgfpathcurveto{\pgfqpoint{1.28cm}{1.549cm}}{\pgfqpoint{1.245cm}{1.563cm}}{\pgfqpoint{1.209cm}{1.563cm}}
\pgfpathcurveto{\pgfqpoint{1.172cm}{1.563cm}}{\pgfqpoint{1.138cm}{1.549cm}}{\pgfqpoint{1.112cm}{1.523cm}}
\pgfpathcurveto{\pgfqpoint{1.087cm}{1.497cm}}{\pgfqpoint{1.072cm}{1.463cm}}{\pgfqpoint{1.072cm}{1.426cm}}
\pgfpathcurveto{\pgfqpoint{1.072cm}{1.39cm}}{\pgfqpoint{1.087cm}{1.355cm}}{\pgfqpoint{1.112cm}{1.329cm}}
\pgfpathcurveto{\pgfqpoint{1.138cm}{1.304cm}}{\pgfqpoint{1.172cm}{1.289cm}}{\pgfqpoint{1.209cm}{1.289cm}}
\pgfpathcurveto{\pgfqpoint{1.245cm}{1.289cm}}{\pgfqpoint{1.28cm}{1.304cm}}{\pgfqpoint{1.305cm}{1.329cm}}
\pgfpathcurveto{\pgfqpoint{1.331cm}{1.355cm}}{\pgfqpoint{1.345cm}{1.39cm}}{\pgfqpoint{1.345cm}{1.426cm}}
\pgfusepath{fill}
\begin{pgfscope}
\pgfsetdash{}{0cm}
\pgfsetlinewidth{0.818mm}
\pgfsetroundcap
\pgfsetmiterlimit{4.0}
\pgfpathmoveto{\pgfqpoint{0.682cm}{0.726cm}}
\pgfpathlineto{\pgfqpoint{0.682cm}{0.097cm}}
\pgfusepath{stroke}
\end{pgfscope}
\end{pgfscope}
\end{pgfscope}
\end{pgfscope}
\end{tikzpicture}}}_{\varepsilon} = \llbracket X_{\varepsilon}^2 \rrbracket,\quad X^{\!\resizebox{0.6em}{!}{
\begin{tikzpicture}
\pgfpathmoveto{\pgfqpoint{0cm}{0cm}}
\pgfpathlineto{\pgfqpoint{1.376cm}{0cm}}
\pgfpathlineto{\pgfqpoint{1.376cm}{1.588cm}}
\pgfpathlineto{\pgfqpoint{0cm}{1.588cm}}
\pgfpathclose
\pgfusepath{clip}
\begin{pgfscope}
\begin{pgfscope}
\pgfpathmoveto{\pgfqpoint{0cm}{0cm}}
\pgfpathlineto{\pgfqpoint{1.376cm}{0cm}}
\pgfpathlineto{\pgfqpoint{1.376cm}{1.588cm}}
\pgfpathlineto{\pgfqpoint{0cm}{1.588cm}}
\pgfpathclose
\pgfusepath{clip}
\begin{pgfscope}
\begin{pgfscope}
\definecolor{eps2pgf_color}{gray}{0.976471}\pgfsetstrokecolor{eps2pgf_color}\pgfsetfillcolor{eps2pgf_color}
\pgfpathmoveto{\pgfqpoint{0cm}{0cm}}
\pgfpathlineto{\pgfqpoint{1.376cm}{0cm}}
\pgfpathlineto{\pgfqpoint{1.376cm}{1.588cm}}
\pgfpathlineto{\pgfqpoint{0cm}{1.588cm}}
\pgfpathclose
\pgfusepath{fill}
\end{pgfscope}
\begin{pgfscope}
\pgfsetdash{}{0cm}
\pgfsetlinewidth{0.818mm}
\pgfsetroundcap
\pgfsetroundjoin
\pgfsetmiterlimit{7.0}
\definecolor{eps2pgf_color}{gray}{0}\pgfsetstrokecolor{eps2pgf_color}\pgfsetfillcolor{eps2pgf_color}
\pgfpathmoveto{\pgfqpoint{0.117cm}{1.476cm}}
\pgfpathlineto{\pgfqpoint{0.682cm}{0.726cm}}
\pgfpathlineto{\pgfqpoint{1.246cm}{1.476cm}}
\pgfusepath{stroke}
\end{pgfscope}
\definecolor{eps2pgf_color}{gray}{0}\pgfsetstrokecolor{eps2pgf_color}\pgfsetfillcolor{eps2pgf_color}
\pgfpathmoveto{\pgfqpoint{0.273cm}{1.451cm}}
\pgfpathcurveto{\pgfqpoint{0.273cm}{1.487cm}}{\pgfqpoint{0.259cm}{1.522cm}}{\pgfqpoint{0.233cm}{1.547cm}}
\pgfpathcurveto{\pgfqpoint{0.207cm}{1.573cm}}{\pgfqpoint{0.173cm}{1.588cm}}{\pgfqpoint{0.137cm}{1.588cm}}
\pgfpathcurveto{\pgfqpoint{0.1cm}{1.588cm}}{\pgfqpoint{0.066cm}{1.573cm}}{\pgfqpoint{0.04cm}{1.547cm}}
\pgfpathcurveto{\pgfqpoint{0.014cm}{1.522cm}}{\pgfqpoint{0cm}{1.487cm}}{\pgfqpoint{0cm}{1.451cm}}
\pgfpathcurveto{\pgfqpoint{0cm}{1.414cm}}{\pgfqpoint{0.014cm}{1.379cm}}{\pgfqpoint{0.04cm}{1.354cm}}
\pgfpathcurveto{\pgfqpoint{0.066cm}{1.328cm}}{\pgfqpoint{0.1cm}{1.314cm}}{\pgfqpoint{0.137cm}{1.314cm}}
\pgfpathcurveto{\pgfqpoint{0.173cm}{1.314cm}}{\pgfqpoint{0.207cm}{1.328cm}}{\pgfqpoint{0.233cm}{1.354cm}}
\pgfpathcurveto{\pgfqpoint{0.259cm}{1.379cm}}{\pgfqpoint{0.273cm}{1.414cm}}{\pgfqpoint{0.273cm}{1.451cm}}
\pgfusepath{fill}
\pgfpathmoveto{\pgfqpoint{1.345cm}{1.426cm}}
\pgfpathcurveto{\pgfqpoint{1.345cm}{1.463cm}}{\pgfqpoint{1.331cm}{1.497cm}}{\pgfqpoint{1.305cm}{1.523cm}}
\pgfpathcurveto{\pgfqpoint{1.28cm}{1.549cm}}{\pgfqpoint{1.245cm}{1.563cm}}{\pgfqpoint{1.209cm}{1.563cm}}
\pgfpathcurveto{\pgfqpoint{1.172cm}{1.563cm}}{\pgfqpoint{1.138cm}{1.549cm}}{\pgfqpoint{1.112cm}{1.523cm}}
\pgfpathcurveto{\pgfqpoint{1.087cm}{1.497cm}}{\pgfqpoint{1.072cm}{1.463cm}}{\pgfqpoint{1.072cm}{1.426cm}}
\pgfpathcurveto{\pgfqpoint{1.072cm}{1.39cm}}{\pgfqpoint{1.087cm}{1.355cm}}{\pgfqpoint{1.112cm}{1.329cm}}
\pgfpathcurveto{\pgfqpoint{1.138cm}{1.304cm}}{\pgfqpoint{1.172cm}{1.289cm}}{\pgfqpoint{1.209cm}{1.289cm}}
\pgfpathcurveto{\pgfqpoint{1.245cm}{1.289cm}}{\pgfqpoint{1.28cm}{1.304cm}}{\pgfqpoint{1.305cm}{1.329cm}}
\pgfpathcurveto{\pgfqpoint{1.331cm}{1.355cm}}{\pgfqpoint{1.345cm}{1.39cm}}{\pgfqpoint{1.345cm}{1.426cm}}
\pgfusepath{fill}
\begin{pgfscope}
\pgfsetdash{}{0cm}
\pgfsetlinewidth{0.818mm}
\pgfsetroundcap
\pgfsetmiterlimit{4.0}
\pgfpathmoveto{\pgfqpoint{0.682cm}{0.726cm}}
\pgfpathlineto{\pgfqpoint{0.682cm}{0.097cm}}
\pgfusepath{stroke}
\end{pgfscope}
\end{pgfscope}
\end{pgfscope}
\end{pgfscope}
\end{tikzpicture}}}_{\varepsilon}(0)=0, \]
\[X^{\!\resizebox{!}{.8em}{
\begin{tikzpicture}
\pgfpathmoveto{\pgfqpoint{0cm}{-0.035cm}}
\pgfpathlineto{\pgfqpoint{1.976cm}{-0.035cm}}
\pgfpathlineto{\pgfqpoint{1.976cm}{1.94cm}}
\pgfpathlineto{\pgfqpoint{0cm}{1.94cm}}
\pgfpathclose
\pgfusepath{clip}
\begin{pgfscope}
\begin{pgfscope}
\pgfpathmoveto{\pgfqpoint{0cm}{-0.035cm}}
\pgfpathlineto{\pgfqpoint{1.976cm}{-0.035cm}}
\pgfpathlineto{\pgfqpoint{1.976cm}{1.94cm}}
\pgfpathlineto{\pgfqpoint{0cm}{1.94cm}}
\pgfpathclose
\pgfusepath{clip}
\begin{pgfscope}
\begin{pgfscope}
\pgfsetdash{}{0cm}
\pgfsetlinewidth{0.818mm}
\pgfsetroundcap
\pgfsetroundjoin
\pgfsetmiterlimit{7.0}
\definecolor{eps2pgf_color}{gray}{0}\pgfsetstrokecolor{eps2pgf_color}\pgfsetfillcolor{eps2pgf_color}
\pgfpathmoveto{\pgfqpoint{0.117cm}{1.815cm}}
\pgfpathlineto{\pgfqpoint{0.682cm}{1.065cm}}
\pgfpathlineto{\pgfqpoint{1.246cm}{1.815cm}}
\pgfusepath{stroke}
\end{pgfscope}
\definecolor{eps2pgf_color}{gray}{0}\pgfsetstrokecolor{eps2pgf_color}\pgfsetfillcolor{eps2pgf_color}
\pgfpathmoveto{\pgfqpoint{0.273cm}{1.789cm}}
\pgfpathcurveto{\pgfqpoint{0.273cm}{1.825cm}}{\pgfqpoint{0.259cm}{1.86cm}}{\pgfqpoint{0.233cm}{1.886cm}}
\pgfpathcurveto{\pgfqpoint{0.207cm}{1.912cm}}{\pgfqpoint{0.173cm}{1.926cm}}{\pgfqpoint{0.137cm}{1.926cm}}
\pgfpathcurveto{\pgfqpoint{0.1cm}{1.926cm}}{\pgfqpoint{0.066cm}{1.912cm}}{\pgfqpoint{0.04cm}{1.886cm}}
\pgfpathcurveto{\pgfqpoint{0.014cm}{1.86cm}}{\pgfqpoint{0cm}{1.825cm}}{\pgfqpoint{0cm}{1.789cm}}
\pgfpathcurveto{\pgfqpoint{0cm}{1.753cm}}{\pgfqpoint{0.014cm}{1.718cm}}{\pgfqpoint{0.04cm}{1.692cm}}
\pgfpathcurveto{\pgfqpoint{0.066cm}{1.667cm}}{\pgfqpoint{0.1cm}{1.652cm}}{\pgfqpoint{0.137cm}{1.652cm}}
\pgfpathcurveto{\pgfqpoint{0.173cm}{1.652cm}}{\pgfqpoint{0.207cm}{1.667cm}}{\pgfqpoint{0.233cm}{1.692cm}}
\pgfpathcurveto{\pgfqpoint{0.259cm}{1.718cm}}{\pgfqpoint{0.273cm}{1.753cm}}{\pgfqpoint{0.273cm}{1.789cm}}
\pgfusepath{fill}
\begin{pgfscope}
\pgfsetdash{}{0cm}
\pgfsetlinewidth{0.818mm}
\pgfsetmiterlimit{7.0}
\pgfpathmoveto{\pgfqpoint{0.682cm}{1.065cm}}
\pgfpathlineto{\pgfqpoint{0.679cm}{1.812cm}}
\pgfusepath{stroke}
\end{pgfscope}
\pgfpathmoveto{\pgfqpoint{0.815cm}{1.793cm}}
\pgfpathcurveto{\pgfqpoint{0.815cm}{1.829cm}}{\pgfqpoint{0.801cm}{1.864cm}}{\pgfqpoint{0.775cm}{1.89cm}}
\pgfpathcurveto{\pgfqpoint{0.75cm}{1.915cm}}{\pgfqpoint{0.715cm}{1.93cm}}{\pgfqpoint{0.679cm}{1.93cm}}
\pgfpathcurveto{\pgfqpoint{0.643cm}{1.93cm}}{\pgfqpoint{0.608cm}{1.915cm}}{\pgfqpoint{0.582cm}{1.89cm}}
\pgfpathcurveto{\pgfqpoint{0.557cm}{1.864cm}}{\pgfqpoint{0.542cm}{1.829cm}}{\pgfqpoint{0.542cm}{1.793cm}}
\pgfpathcurveto{\pgfqpoint{0.542cm}{1.756cm}}{\pgfqpoint{0.557cm}{1.722cm}}{\pgfqpoint{0.582cm}{1.696cm}}
\pgfpathcurveto{\pgfqpoint{0.608cm}{1.67cm}}{\pgfqpoint{0.643cm}{1.656cm}}{\pgfqpoint{0.679cm}{1.656cm}}
\pgfpathcurveto{\pgfqpoint{0.715cm}{1.656cm}}{\pgfqpoint{0.75cm}{1.67cm}}{\pgfqpoint{0.775cm}{1.696cm}}
\pgfpathcurveto{\pgfqpoint{0.801cm}{1.722cm}}{\pgfqpoint{0.815cm}{1.756cm}}{\pgfqpoint{0.815cm}{1.793cm}}
\pgfusepath{fill}
\pgfpathmoveto{\pgfqpoint{1.345cm}{1.765cm}}
\pgfpathcurveto{\pgfqpoint{1.345cm}{1.801cm}}{\pgfqpoint{1.331cm}{1.836cm}}{\pgfqpoint{1.305cm}{1.862cm}}
\pgfpathcurveto{\pgfqpoint{1.28cm}{1.887cm}}{\pgfqpoint{1.245cm}{1.902cm}}{\pgfqpoint{1.209cm}{1.902cm}}
\pgfpathcurveto{\pgfqpoint{1.172cm}{1.902cm}}{\pgfqpoint{1.138cm}{1.887cm}}{\pgfqpoint{1.112cm}{1.862cm}}
\pgfpathcurveto{\pgfqpoint{1.087cm}{1.836cm}}{\pgfqpoint{1.072cm}{1.801cm}}{\pgfqpoint{1.072cm}{1.765cm}}
\pgfpathcurveto{\pgfqpoint{1.072cm}{1.728cm}}{\pgfqpoint{1.087cm}{1.694cm}}{\pgfqpoint{1.112cm}{1.668cm}}
\pgfpathcurveto{\pgfqpoint{1.138cm}{1.642cm}}{\pgfqpoint{1.172cm}{1.628cm}}{\pgfqpoint{1.209cm}{1.628cm}}
\pgfpathcurveto{\pgfqpoint{1.245cm}{1.628cm}}{\pgfqpoint{1.28cm}{1.642cm}}{\pgfqpoint{1.305cm}{1.668cm}}
\pgfpathcurveto{\pgfqpoint{1.331cm}{1.694cm}}{\pgfqpoint{1.345cm}{1.728cm}}{\pgfqpoint{1.345cm}{1.765cm}}
\pgfusepath{fill}
\begin{pgfscope}
\pgfsetdash{}{0cm}
\pgfsetlinewidth{0.818mm}
\pgfsetroundcap
\pgfsetroundjoin
\pgfsetmiterlimit{7.0}
\pgfpathmoveto{\pgfqpoint{0.682cm}{1.065cm}}
\pgfpathlineto{\pgfqpoint{1.246cm}{0.315cm}}
\pgfpathlineto{\pgfqpoint{1.811cm}{1.065cm}}
\pgfusepath{stroke}
\end{pgfscope}
\pgfpathmoveto{\pgfqpoint{1.948cm}{1.065cm}}
\pgfpathcurveto{\pgfqpoint{1.948cm}{1.101cm}}{\pgfqpoint{1.933cm}{1.136cm}}{\pgfqpoint{1.907cm}{1.162cm}}
\pgfpathcurveto{\pgfqpoint{1.882cm}{1.187cm}}{\pgfqpoint{1.847cm}{1.202cm}}{\pgfqpoint{1.811cm}{1.202cm}}
\pgfpathcurveto{\pgfqpoint{1.775cm}{1.202cm}}{\pgfqpoint{1.74cm}{1.187cm}}{\pgfqpoint{1.714cm}{1.162cm}}
\pgfpathcurveto{\pgfqpoint{1.689cm}{1.136cm}}{\pgfqpoint{1.674cm}{1.101cm}}{\pgfqpoint{1.674cm}{1.065cm}}
\pgfpathcurveto{\pgfqpoint{1.674cm}{1.029cm}}{\pgfqpoint{1.689cm}{0.994cm}}{\pgfqpoint{1.714cm}{0.968cm}}
\pgfpathcurveto{\pgfqpoint{1.74cm}{0.942cm}}{\pgfqpoint{1.775cm}{0.928cm}}{\pgfqpoint{1.811cm}{0.928cm}}
\pgfpathcurveto{\pgfqpoint{1.847cm}{0.928cm}}{\pgfqpoint{1.882cm}{0.942cm}}{\pgfqpoint{1.907cm}{0.968cm}}
\pgfpathcurveto{\pgfqpoint{1.933cm}{0.994cm}}{\pgfqpoint{1.948cm}{1.029cm}}{\pgfqpoint{1.948cm}{1.065cm}}
\pgfusepath{fill}
\begin{pgfscope}
\pgfsetdash{}{0cm}
\pgfsetlinewidth{0.818mm}
\pgfsetmiterlimit{4.0}
\pgfpathmoveto{\pgfqpoint{1.383cm}{0.178cm}}
\pgfpathcurveto{\pgfqpoint{1.383cm}{0.214cm}}{\pgfqpoint{1.369cm}{0.249cm}}{\pgfqpoint{1.343cm}{0.275cm}}
\pgfpathcurveto{\pgfqpoint{1.317cm}{0.3cm}}{\pgfqpoint{1.283cm}{0.315cm}}{\pgfqpoint{1.246cm}{0.315cm}}
\pgfpathcurveto{\pgfqpoint{1.21cm}{0.315cm}}{\pgfqpoint{1.175cm}{0.3cm}}{\pgfqpoint{1.15cm}{0.275cm}}
\pgfpathcurveto{\pgfqpoint{1.124cm}{0.249cm}}{\pgfqpoint{1.11cm}{0.214cm}}{\pgfqpoint{1.11cm}{0.178cm}}
\pgfpathcurveto{\pgfqpoint{1.11cm}{0.141cm}}{\pgfqpoint{1.124cm}{0.107cm}}{\pgfqpoint{1.15cm}{0.081cm}}
\pgfpathcurveto{\pgfqpoint{1.175cm}{0.055cm}}{\pgfqpoint{1.21cm}{0.041cm}}{\pgfqpoint{1.246cm}{0.041cm}}
\pgfpathcurveto{\pgfqpoint{1.283cm}{0.041cm}}{\pgfqpoint{1.317cm}{0.055cm}}{\pgfqpoint{1.343cm}{0.081cm}}
\pgfpathcurveto{\pgfqpoint{1.369cm}{0.107cm}}{\pgfqpoint{1.383cm}{0.141cm}}{\pgfqpoint{1.383cm}{0.178cm}}
\pgfusepath{stroke}
\end{pgfscope}
\end{pgfscope}
\end{pgfscope}
\end{pgfscope}
\end{tikzpicture}}}=\lim_{\varepsilon\to 0}X^{\!\resizebox{0.6em}{!}{
\begin{tikzpicture}
\pgfpathmoveto{\pgfqpoint{0cm}{-0.035cm}}
\pgfpathlineto{\pgfqpoint{1.376cm}{-0.035cm}}
\pgfpathlineto{\pgfqpoint{1.376cm}{1.552cm}}
\pgfpathlineto{\pgfqpoint{0cm}{1.552cm}}
\pgfpathclose
\pgfusepath{clip}
\begin{pgfscope}
\begin{pgfscope}
\pgfpathmoveto{\pgfqpoint{0cm}{-0.035cm}}
\pgfpathlineto{\pgfqpoint{1.376cm}{-0.035cm}}
\pgfpathlineto{\pgfqpoint{1.376cm}{1.552cm}}
\pgfpathlineto{\pgfqpoint{0cm}{1.552cm}}
\pgfpathclose
\pgfusepath{clip}
\begin{pgfscope}
\begin{pgfscope}
\pgfsetdash{}{0cm}
\pgfsetlinewidth{0.818mm}
\pgfsetroundcap
\pgfsetroundjoin
\pgfsetmiterlimit{7.0}
\definecolor{eps2pgf_color}{gray}{0}\pgfsetstrokecolor{eps2pgf_color}\pgfsetfillcolor{eps2pgf_color}
\pgfpathmoveto{\pgfqpoint{0.117cm}{1.421cm}}
\pgfpathlineto{\pgfqpoint{0.682cm}{0.671cm}}
\pgfpathlineto{\pgfqpoint{1.246cm}{1.421cm}}
\pgfusepath{stroke}
\end{pgfscope}
\definecolor{eps2pgf_color}{gray}{0}\pgfsetstrokecolor{eps2pgf_color}\pgfsetfillcolor{eps2pgf_color}
\pgfpathmoveto{\pgfqpoint{0.273cm}{1.395cm}}
\pgfpathcurveto{\pgfqpoint{0.273cm}{1.432cm}}{\pgfqpoint{0.259cm}{1.467cm}}{\pgfqpoint{0.233cm}{1.492cm}}
\pgfpathcurveto{\pgfqpoint{0.207cm}{1.518cm}}{\pgfqpoint{0.173cm}{1.532cm}}{\pgfqpoint{0.137cm}{1.532cm}}
\pgfpathcurveto{\pgfqpoint{0.1cm}{1.532cm}}{\pgfqpoint{0.066cm}{1.518cm}}{\pgfqpoint{0.04cm}{1.492cm}}
\pgfpathcurveto{\pgfqpoint{0.014cm}{1.467cm}}{\pgfqpoint{0cm}{1.432cm}}{\pgfqpoint{0cm}{1.395cm}}
\pgfpathcurveto{\pgfqpoint{0cm}{1.359cm}}{\pgfqpoint{0.014cm}{1.324cm}}{\pgfqpoint{0.04cm}{1.299cm}}
\pgfpathcurveto{\pgfqpoint{0.066cm}{1.273cm}}{\pgfqpoint{0.1cm}{1.258cm}}{\pgfqpoint{0.137cm}{1.258cm}}
\pgfpathcurveto{\pgfqpoint{0.173cm}{1.258cm}}{\pgfqpoint{0.207cm}{1.273cm}}{\pgfqpoint{0.233cm}{1.299cm}}
\pgfpathcurveto{\pgfqpoint{0.259cm}{1.324cm}}{\pgfqpoint{0.273cm}{1.359cm}}{\pgfqpoint{0.273cm}{1.395cm}}
\pgfusepath{fill}
\begin{pgfscope}
\pgfsetdash{}{0cm}
\pgfsetlinewidth{0.818mm}
\pgfsetmiterlimit{7.0}
\pgfpathmoveto{\pgfqpoint{0.682cm}{0.671cm}}
\pgfpathlineto{\pgfqpoint{0.679cm}{1.418cm}}
\pgfusepath{stroke}
\end{pgfscope}
\pgfpathmoveto{\pgfqpoint{0.815cm}{1.399cm}}
\pgfpathcurveto{\pgfqpoint{0.815cm}{1.435cm}}{\pgfqpoint{0.801cm}{1.47cm}}{\pgfqpoint{0.775cm}{1.496cm}}
\pgfpathcurveto{\pgfqpoint{0.75cm}{1.521cm}}{\pgfqpoint{0.715cm}{1.536cm}}{\pgfqpoint{0.679cm}{1.536cm}}
\pgfpathcurveto{\pgfqpoint{0.643cm}{1.536cm}}{\pgfqpoint{0.608cm}{1.521cm}}{\pgfqpoint{0.582cm}{1.496cm}}
\pgfpathcurveto{\pgfqpoint{0.557cm}{1.47cm}}{\pgfqpoint{0.542cm}{1.435cm}}{\pgfqpoint{0.542cm}{1.399cm}}
\pgfpathcurveto{\pgfqpoint{0.542cm}{1.363cm}}{\pgfqpoint{0.557cm}{1.328cm}}{\pgfqpoint{0.582cm}{1.302cm}}
\pgfpathcurveto{\pgfqpoint{0.608cm}{1.276cm}}{\pgfqpoint{0.643cm}{1.262cm}}{\pgfqpoint{0.679cm}{1.262cm}}
\pgfpathcurveto{\pgfqpoint{0.715cm}{1.262cm}}{\pgfqpoint{0.75cm}{1.276cm}}{\pgfqpoint{0.775cm}{1.302cm}}
\pgfpathcurveto{\pgfqpoint{0.801cm}{1.328cm}}{\pgfqpoint{0.815cm}{1.363cm}}{\pgfqpoint{0.815cm}{1.399cm}}
\pgfusepath{fill}
\pgfpathmoveto{\pgfqpoint{1.345cm}{1.371cm}}
\pgfpathcurveto{\pgfqpoint{1.345cm}{1.408cm}}{\pgfqpoint{1.331cm}{1.442cm}}{\pgfqpoint{1.305cm}{1.468cm}}
\pgfpathcurveto{\pgfqpoint{1.28cm}{1.494cm}}{\pgfqpoint{1.245cm}{1.508cm}}{\pgfqpoint{1.209cm}{1.508cm}}
\pgfpathcurveto{\pgfqpoint{1.172cm}{1.508cm}}{\pgfqpoint{1.138cm}{1.494cm}}{\pgfqpoint{1.112cm}{1.468cm}}
\pgfpathcurveto{\pgfqpoint{1.087cm}{1.442cm}}{\pgfqpoint{1.072cm}{1.408cm}}{\pgfqpoint{1.072cm}{1.371cm}}
\pgfpathcurveto{\pgfqpoint{1.072cm}{1.335cm}}{\pgfqpoint{1.087cm}{1.3cm}}{\pgfqpoint{1.112cm}{1.274cm}}
\pgfpathcurveto{\pgfqpoint{1.138cm}{1.249cm}}{\pgfqpoint{1.172cm}{1.234cm}}{\pgfqpoint{1.209cm}{1.234cm}}
\pgfpathcurveto{\pgfqpoint{1.245cm}{1.234cm}}{\pgfqpoint{1.28cm}{1.249cm}}{\pgfqpoint{1.305cm}{1.274cm}}
\pgfpathcurveto{\pgfqpoint{1.331cm}{1.3cm}}{\pgfqpoint{1.345cm}{1.335cm}}{\pgfqpoint{1.345cm}{1.371cm}}
\pgfusepath{fill}
\begin{pgfscope}
\pgfsetdash{}{0cm}
\pgfsetlinewidth{0.818mm}
\pgfsetroundcap
\pgfsetmiterlimit{4.0}
\pgfpathmoveto{\pgfqpoint{0.682cm}{0.671cm}}
\pgfpathlineto{\pgfqpoint{0.682cm}{0.042cm}}
\pgfusepath{stroke}
\end{pgfscope}
\end{pgfscope}
\end{pgfscope}
\end{pgfscope}
\end{tikzpicture}}}_{\varepsilon}\circ X_{\varepsilon},\qquad X^{\!\resizebox{!}{.8em}{
\begin{tikzpicture}
\pgfpathmoveto{\pgfqpoint{0cm}{-0.035cm}}
\pgfpathlineto{\pgfqpoint{1.976cm}{-0.035cm}}
\pgfpathlineto{\pgfqpoint{1.976cm}{1.94cm}}
\pgfpathlineto{\pgfqpoint{0cm}{1.94cm}}
\pgfpathclose
\pgfusepath{clip}
\begin{pgfscope}
\begin{pgfscope}
\pgfpathmoveto{\pgfqpoint{0cm}{-0.035cm}}
\pgfpathlineto{\pgfqpoint{1.976cm}{-0.035cm}}
\pgfpathlineto{\pgfqpoint{1.976cm}{1.94cm}}
\pgfpathlineto{\pgfqpoint{0cm}{1.94cm}}
\pgfpathclose
\pgfusepath{clip}
\begin{pgfscope}
\begin{pgfscope}
\pgfsetdash{}{0cm}
\pgfsetlinewidth{0.818mm}
\pgfsetroundcap
\pgfsetroundjoin
\pgfsetmiterlimit{7.0}
\definecolor{eps2pgf_color}{gray}{0}\pgfsetstrokecolor{eps2pgf_color}\pgfsetfillcolor{eps2pgf_color}
\pgfpathmoveto{\pgfqpoint{0.117cm}{1.815cm}}
\pgfpathlineto{\pgfqpoint{0.682cm}{1.065cm}}
\pgfpathlineto{\pgfqpoint{1.246cm}{1.815cm}}
\pgfusepath{stroke}
\end{pgfscope}
\definecolor{eps2pgf_color}{gray}{0}\pgfsetstrokecolor{eps2pgf_color}\pgfsetfillcolor{eps2pgf_color}
\pgfpathmoveto{\pgfqpoint{0.273cm}{1.789cm}}
\pgfpathcurveto{\pgfqpoint{0.273cm}{1.825cm}}{\pgfqpoint{0.259cm}{1.86cm}}{\pgfqpoint{0.233cm}{1.886cm}}
\pgfpathcurveto{\pgfqpoint{0.207cm}{1.912cm}}{\pgfqpoint{0.173cm}{1.926cm}}{\pgfqpoint{0.137cm}{1.926cm}}
\pgfpathcurveto{\pgfqpoint{0.1cm}{1.926cm}}{\pgfqpoint{0.066cm}{1.912cm}}{\pgfqpoint{0.04cm}{1.886cm}}
\pgfpathcurveto{\pgfqpoint{0.014cm}{1.86cm}}{\pgfqpoint{0cm}{1.825cm}}{\pgfqpoint{0cm}{1.789cm}}
\pgfpathcurveto{\pgfqpoint{0cm}{1.753cm}}{\pgfqpoint{0.014cm}{1.718cm}}{\pgfqpoint{0.04cm}{1.692cm}}
\pgfpathcurveto{\pgfqpoint{0.066cm}{1.667cm}}{\pgfqpoint{0.1cm}{1.652cm}}{\pgfqpoint{0.137cm}{1.652cm}}
\pgfpathcurveto{\pgfqpoint{0.173cm}{1.652cm}}{\pgfqpoint{0.207cm}{1.667cm}}{\pgfqpoint{0.233cm}{1.692cm}}
\pgfpathcurveto{\pgfqpoint{0.259cm}{1.718cm}}{\pgfqpoint{0.273cm}{1.753cm}}{\pgfqpoint{0.273cm}{1.789cm}}
\pgfusepath{fill}
\pgfpathmoveto{\pgfqpoint{1.345cm}{1.765cm}}
\pgfpathcurveto{\pgfqpoint{1.345cm}{1.801cm}}{\pgfqpoint{1.331cm}{1.836cm}}{\pgfqpoint{1.305cm}{1.862cm}}
\pgfpathcurveto{\pgfqpoint{1.28cm}{1.887cm}}{\pgfqpoint{1.245cm}{1.902cm}}{\pgfqpoint{1.209cm}{1.902cm}}
\pgfpathcurveto{\pgfqpoint{1.172cm}{1.902cm}}{\pgfqpoint{1.138cm}{1.887cm}}{\pgfqpoint{1.112cm}{1.862cm}}
\pgfpathcurveto{\pgfqpoint{1.087cm}{1.836cm}}{\pgfqpoint{1.072cm}{1.801cm}}{\pgfqpoint{1.072cm}{1.765cm}}
\pgfpathcurveto{\pgfqpoint{1.072cm}{1.728cm}}{\pgfqpoint{1.087cm}{1.694cm}}{\pgfqpoint{1.112cm}{1.668cm}}
\pgfpathcurveto{\pgfqpoint{1.138cm}{1.642cm}}{\pgfqpoint{1.172cm}{1.628cm}}{\pgfqpoint{1.209cm}{1.628cm}}
\pgfpathcurveto{\pgfqpoint{1.245cm}{1.628cm}}{\pgfqpoint{1.28cm}{1.642cm}}{\pgfqpoint{1.305cm}{1.668cm}}
\pgfpathcurveto{\pgfqpoint{1.331cm}{1.694cm}}{\pgfqpoint{1.345cm}{1.728cm}}{\pgfqpoint{1.345cm}{1.765cm}}
\pgfusepath{fill}
\begin{pgfscope}
\pgfsetdash{}{0cm}
\pgfsetlinewidth{0.818mm}
\pgfsetroundcap
\pgfsetroundjoin
\pgfsetmiterlimit{7.0}
\pgfpathmoveto{\pgfqpoint{0.682cm}{1.065cm}}
\pgfpathlineto{\pgfqpoint{1.246cm}{0.315cm}}
\pgfpathlineto{\pgfqpoint{1.811cm}{1.065cm}}
\pgfusepath{stroke}
\end{pgfscope}
\pgfpathmoveto{\pgfqpoint{1.948cm}{1.065cm}}
\pgfpathcurveto{\pgfqpoint{1.948cm}{1.101cm}}{\pgfqpoint{1.933cm}{1.136cm}}{\pgfqpoint{1.907cm}{1.162cm}}
\pgfpathcurveto{\pgfqpoint{1.882cm}{1.187cm}}{\pgfqpoint{1.847cm}{1.202cm}}{\pgfqpoint{1.811cm}{1.202cm}}
\pgfpathcurveto{\pgfqpoint{1.775cm}{1.202cm}}{\pgfqpoint{1.74cm}{1.187cm}}{\pgfqpoint{1.714cm}{1.162cm}}
\pgfpathcurveto{\pgfqpoint{1.689cm}{1.136cm}}{\pgfqpoint{1.674cm}{1.101cm}}{\pgfqpoint{1.674cm}{1.065cm}}
\pgfpathcurveto{\pgfqpoint{1.674cm}{1.029cm}}{\pgfqpoint{1.689cm}{0.994cm}}{\pgfqpoint{1.714cm}{0.968cm}}
\pgfpathcurveto{\pgfqpoint{1.74cm}{0.942cm}}{\pgfqpoint{1.775cm}{0.928cm}}{\pgfqpoint{1.811cm}{0.928cm}}
\pgfpathcurveto{\pgfqpoint{1.847cm}{0.928cm}}{\pgfqpoint{1.882cm}{0.942cm}}{\pgfqpoint{1.907cm}{0.968cm}}
\pgfpathcurveto{\pgfqpoint{1.933cm}{0.994cm}}{\pgfqpoint{1.948cm}{1.029cm}}{\pgfqpoint{1.948cm}{1.065cm}}
\pgfusepath{fill}
\begin{pgfscope}
\pgfsetdash{}{0cm}
\pgfsetlinewidth{0.818mm}
\pgfsetmiterlimit{7.0}
\pgfpathmoveto{\pgfqpoint{1.246cm}{0.315cm}}
\pgfpathlineto{\pgfqpoint{1.244cm}{1.061cm}}
\pgfusepath{stroke}
\end{pgfscope}
\pgfpathmoveto{\pgfqpoint{1.38cm}{1.065cm}}
\pgfpathcurveto{\pgfqpoint{1.38cm}{1.101cm}}{\pgfqpoint{1.366cm}{1.136cm}}{\pgfqpoint{1.34cm}{1.162cm}}
\pgfpathcurveto{\pgfqpoint{1.315cm}{1.187cm}}{\pgfqpoint{1.28cm}{1.202cm}}{\pgfqpoint{1.244cm}{1.202cm}}
\pgfpathcurveto{\pgfqpoint{1.207cm}{1.202cm}}{\pgfqpoint{1.173cm}{1.187cm}}{\pgfqpoint{1.147cm}{1.162cm}}
\pgfpathcurveto{\pgfqpoint{1.121cm}{1.136cm}}{\pgfqpoint{1.107cm}{1.101cm}}{\pgfqpoint{1.107cm}{1.065cm}}
\pgfpathcurveto{\pgfqpoint{1.107cm}{1.029cm}}{\pgfqpoint{1.121cm}{0.994cm}}{\pgfqpoint{1.147cm}{0.968cm}}
\pgfpathcurveto{\pgfqpoint{1.173cm}{0.942cm}}{\pgfqpoint{1.207cm}{0.928cm}}{\pgfqpoint{1.244cm}{0.928cm}}
\pgfpathcurveto{\pgfqpoint{1.28cm}{0.928cm}}{\pgfqpoint{1.315cm}{0.942cm}}{\pgfqpoint{1.34cm}{0.968cm}}
\pgfpathcurveto{\pgfqpoint{1.366cm}{0.994cm}}{\pgfqpoint{1.38cm}{1.029cm}}{\pgfqpoint{1.38cm}{1.065cm}}
\pgfusepath{fill}
\begin{pgfscope}
\pgfsetdash{}{0cm}
\pgfsetlinewidth{0.818mm}
\pgfsetmiterlimit{4.0}
\pgfpathmoveto{\pgfqpoint{1.383cm}{0.178cm}}
\pgfpathcurveto{\pgfqpoint{1.383cm}{0.214cm}}{\pgfqpoint{1.369cm}{0.249cm}}{\pgfqpoint{1.343cm}{0.275cm}}
\pgfpathcurveto{\pgfqpoint{1.317cm}{0.3cm}}{\pgfqpoint{1.283cm}{0.315cm}}{\pgfqpoint{1.246cm}{0.315cm}}
\pgfpathcurveto{\pgfqpoint{1.21cm}{0.315cm}}{\pgfqpoint{1.175cm}{0.3cm}}{\pgfqpoint{1.15cm}{0.275cm}}
\pgfpathcurveto{\pgfqpoint{1.124cm}{0.249cm}}{\pgfqpoint{1.11cm}{0.214cm}}{\pgfqpoint{1.11cm}{0.178cm}}
\pgfpathcurveto{\pgfqpoint{1.11cm}{0.141cm}}{\pgfqpoint{1.124cm}{0.107cm}}{\pgfqpoint{1.15cm}{0.081cm}}
\pgfpathcurveto{\pgfqpoint{1.175cm}{0.055cm}}{\pgfqpoint{1.21cm}{0.041cm}}{\pgfqpoint{1.246cm}{0.041cm}}
\pgfpathcurveto{\pgfqpoint{1.283cm}{0.041cm}}{\pgfqpoint{1.317cm}{0.055cm}}{\pgfqpoint{1.343cm}{0.081cm}}
\pgfpathcurveto{\pgfqpoint{1.369cm}{0.107cm}}{\pgfqpoint{1.383cm}{0.141cm}}{\pgfqpoint{1.383cm}{0.178cm}}
\pgfusepath{stroke}
\end{pgfscope}
\end{pgfscope}
\end{pgfscope}
\end{pgfscope}
\end{tikzpicture}}} = \lim_{\varepsilon\to 0}X^{\!\resizebox{0.6em}{!}{
\begin{tikzpicture}
\pgfpathmoveto{\pgfqpoint{0cm}{0cm}}
\pgfpathlineto{\pgfqpoint{1.376cm}{0cm}}
\pgfpathlineto{\pgfqpoint{1.376cm}{1.588cm}}
\pgfpathlineto{\pgfqpoint{0cm}{1.588cm}}
\pgfpathclose
\pgfusepath{clip}
\begin{pgfscope}
\begin{pgfscope}
\pgfpathmoveto{\pgfqpoint{0cm}{0cm}}
\pgfpathlineto{\pgfqpoint{1.376cm}{0cm}}
\pgfpathlineto{\pgfqpoint{1.376cm}{1.588cm}}
\pgfpathlineto{\pgfqpoint{0cm}{1.588cm}}
\pgfpathclose
\pgfusepath{clip}
\begin{pgfscope}
\begin{pgfscope}
\definecolor{eps2pgf_color}{gray}{0.976471}\pgfsetstrokecolor{eps2pgf_color}\pgfsetfillcolor{eps2pgf_color}
\pgfpathmoveto{\pgfqpoint{0cm}{0cm}}
\pgfpathlineto{\pgfqpoint{1.376cm}{0cm}}
\pgfpathlineto{\pgfqpoint{1.376cm}{1.588cm}}
\pgfpathlineto{\pgfqpoint{0cm}{1.588cm}}
\pgfpathclose
\pgfusepath{fill}
\end{pgfscope}
\begin{pgfscope}
\pgfsetdash{}{0cm}
\pgfsetlinewidth{0.818mm}
\pgfsetroundcap
\pgfsetroundjoin
\pgfsetmiterlimit{7.0}
\definecolor{eps2pgf_color}{gray}{0}\pgfsetstrokecolor{eps2pgf_color}\pgfsetfillcolor{eps2pgf_color}
\pgfpathmoveto{\pgfqpoint{0.117cm}{1.476cm}}
\pgfpathlineto{\pgfqpoint{0.682cm}{0.726cm}}
\pgfpathlineto{\pgfqpoint{1.246cm}{1.476cm}}
\pgfusepath{stroke}
\end{pgfscope}
\definecolor{eps2pgf_color}{gray}{0}\pgfsetstrokecolor{eps2pgf_color}\pgfsetfillcolor{eps2pgf_color}
\pgfpathmoveto{\pgfqpoint{0.273cm}{1.451cm}}
\pgfpathcurveto{\pgfqpoint{0.273cm}{1.487cm}}{\pgfqpoint{0.259cm}{1.522cm}}{\pgfqpoint{0.233cm}{1.547cm}}
\pgfpathcurveto{\pgfqpoint{0.207cm}{1.573cm}}{\pgfqpoint{0.173cm}{1.588cm}}{\pgfqpoint{0.137cm}{1.588cm}}
\pgfpathcurveto{\pgfqpoint{0.1cm}{1.588cm}}{\pgfqpoint{0.066cm}{1.573cm}}{\pgfqpoint{0.04cm}{1.547cm}}
\pgfpathcurveto{\pgfqpoint{0.014cm}{1.522cm}}{\pgfqpoint{0cm}{1.487cm}}{\pgfqpoint{0cm}{1.451cm}}
\pgfpathcurveto{\pgfqpoint{0cm}{1.414cm}}{\pgfqpoint{0.014cm}{1.379cm}}{\pgfqpoint{0.04cm}{1.354cm}}
\pgfpathcurveto{\pgfqpoint{0.066cm}{1.328cm}}{\pgfqpoint{0.1cm}{1.314cm}}{\pgfqpoint{0.137cm}{1.314cm}}
\pgfpathcurveto{\pgfqpoint{0.173cm}{1.314cm}}{\pgfqpoint{0.207cm}{1.328cm}}{\pgfqpoint{0.233cm}{1.354cm}}
\pgfpathcurveto{\pgfqpoint{0.259cm}{1.379cm}}{\pgfqpoint{0.273cm}{1.414cm}}{\pgfqpoint{0.273cm}{1.451cm}}
\pgfusepath{fill}
\pgfpathmoveto{\pgfqpoint{1.345cm}{1.426cm}}
\pgfpathcurveto{\pgfqpoint{1.345cm}{1.463cm}}{\pgfqpoint{1.331cm}{1.497cm}}{\pgfqpoint{1.305cm}{1.523cm}}
\pgfpathcurveto{\pgfqpoint{1.28cm}{1.549cm}}{\pgfqpoint{1.245cm}{1.563cm}}{\pgfqpoint{1.209cm}{1.563cm}}
\pgfpathcurveto{\pgfqpoint{1.172cm}{1.563cm}}{\pgfqpoint{1.138cm}{1.549cm}}{\pgfqpoint{1.112cm}{1.523cm}}
\pgfpathcurveto{\pgfqpoint{1.087cm}{1.497cm}}{\pgfqpoint{1.072cm}{1.463cm}}{\pgfqpoint{1.072cm}{1.426cm}}
\pgfpathcurveto{\pgfqpoint{1.072cm}{1.39cm}}{\pgfqpoint{1.087cm}{1.355cm}}{\pgfqpoint{1.112cm}{1.329cm}}
\pgfpathcurveto{\pgfqpoint{1.138cm}{1.304cm}}{\pgfqpoint{1.172cm}{1.289cm}}{\pgfqpoint{1.209cm}{1.289cm}}
\pgfpathcurveto{\pgfqpoint{1.245cm}{1.289cm}}{\pgfqpoint{1.28cm}{1.304cm}}{\pgfqpoint{1.305cm}{1.329cm}}
\pgfpathcurveto{\pgfqpoint{1.331cm}{1.355cm}}{\pgfqpoint{1.345cm}{1.39cm}}{\pgfqpoint{1.345cm}{1.426cm}}
\pgfusepath{fill}
\begin{pgfscope}
\pgfsetdash{}{0cm}
\pgfsetlinewidth{0.818mm}
\pgfsetroundcap
\pgfsetmiterlimit{4.0}
\pgfpathmoveto{\pgfqpoint{0.682cm}{0.726cm}}
\pgfpathlineto{\pgfqpoint{0.682cm}{0.097cm}}
\pgfusepath{stroke}
\end{pgfscope}
\end{pgfscope}
\end{pgfscope}
\end{pgfscope}
\end{tikzpicture}}}_{\varepsilon} \circ \llbracket X_{\varepsilon}^2 \rrbracket -\frac{b_{\varepsilon}}{3}, \qquad
   X^{\!\resizebox{!}{.8em}{
\begin{tikzpicture}
\pgfpathmoveto{\pgfqpoint{0cm}{-0.035cm}}
\pgfpathlineto{\pgfqpoint{1.976cm}{-0.035cm}}
\pgfpathlineto{\pgfqpoint{1.976cm}{1.94cm}}
\pgfpathlineto{\pgfqpoint{0cm}{1.94cm}}
\pgfpathclose
\pgfusepath{clip}
\begin{pgfscope}
\begin{pgfscope}
\pgfpathmoveto{\pgfqpoint{0cm}{-0.035cm}}
\pgfpathlineto{\pgfqpoint{1.976cm}{-0.035cm}}
\pgfpathlineto{\pgfqpoint{1.976cm}{1.94cm}}
\pgfpathlineto{\pgfqpoint{0cm}{1.94cm}}
\pgfpathclose
\pgfusepath{clip}
\begin{pgfscope}
\begin{pgfscope}
\pgfsetdash{}{0cm}
\pgfsetlinewidth{0.818mm}
\pgfsetroundcap
\pgfsetroundjoin
\pgfsetmiterlimit{7.0}
\definecolor{eps2pgf_color}{gray}{0}\pgfsetstrokecolor{eps2pgf_color}\pgfsetfillcolor{eps2pgf_color}
\pgfpathmoveto{\pgfqpoint{0.117cm}{1.815cm}}
\pgfpathlineto{\pgfqpoint{0.682cm}{1.065cm}}
\pgfpathlineto{\pgfqpoint{1.246cm}{1.815cm}}
\pgfusepath{stroke}
\end{pgfscope}
\definecolor{eps2pgf_color}{gray}{0}\pgfsetstrokecolor{eps2pgf_color}\pgfsetfillcolor{eps2pgf_color}
\pgfpathmoveto{\pgfqpoint{0.273cm}{1.789cm}}
\pgfpathcurveto{\pgfqpoint{0.273cm}{1.825cm}}{\pgfqpoint{0.259cm}{1.86cm}}{\pgfqpoint{0.233cm}{1.886cm}}
\pgfpathcurveto{\pgfqpoint{0.207cm}{1.912cm}}{\pgfqpoint{0.173cm}{1.926cm}}{\pgfqpoint{0.137cm}{1.926cm}}
\pgfpathcurveto{\pgfqpoint{0.1cm}{1.926cm}}{\pgfqpoint{0.066cm}{1.912cm}}{\pgfqpoint{0.04cm}{1.886cm}}
\pgfpathcurveto{\pgfqpoint{0.014cm}{1.86cm}}{\pgfqpoint{0cm}{1.825cm}}{\pgfqpoint{0cm}{1.789cm}}
\pgfpathcurveto{\pgfqpoint{0cm}{1.753cm}}{\pgfqpoint{0.014cm}{1.718cm}}{\pgfqpoint{0.04cm}{1.692cm}}
\pgfpathcurveto{\pgfqpoint{0.066cm}{1.667cm}}{\pgfqpoint{0.1cm}{1.652cm}}{\pgfqpoint{0.137cm}{1.652cm}}
\pgfpathcurveto{\pgfqpoint{0.173cm}{1.652cm}}{\pgfqpoint{0.207cm}{1.667cm}}{\pgfqpoint{0.233cm}{1.692cm}}
\pgfpathcurveto{\pgfqpoint{0.259cm}{1.718cm}}{\pgfqpoint{0.273cm}{1.753cm}}{\pgfqpoint{0.273cm}{1.789cm}}
\pgfusepath{fill}
\begin{pgfscope}
\pgfsetdash{}{0cm}
\pgfsetlinewidth{0.818mm}
\pgfsetmiterlimit{7.0}
\pgfpathmoveto{\pgfqpoint{0.682cm}{1.065cm}}
\pgfpathlineto{\pgfqpoint{0.679cm}{1.812cm}}
\pgfusepath{stroke}
\end{pgfscope}
\pgfpathmoveto{\pgfqpoint{0.815cm}{1.793cm}}
\pgfpathcurveto{\pgfqpoint{0.815cm}{1.829cm}}{\pgfqpoint{0.801cm}{1.864cm}}{\pgfqpoint{0.775cm}{1.89cm}}
\pgfpathcurveto{\pgfqpoint{0.75cm}{1.915cm}}{\pgfqpoint{0.715cm}{1.93cm}}{\pgfqpoint{0.679cm}{1.93cm}}
\pgfpathcurveto{\pgfqpoint{0.643cm}{1.93cm}}{\pgfqpoint{0.608cm}{1.915cm}}{\pgfqpoint{0.582cm}{1.89cm}}
\pgfpathcurveto{\pgfqpoint{0.557cm}{1.864cm}}{\pgfqpoint{0.542cm}{1.829cm}}{\pgfqpoint{0.542cm}{1.793cm}}
\pgfpathcurveto{\pgfqpoint{0.542cm}{1.756cm}}{\pgfqpoint{0.557cm}{1.722cm}}{\pgfqpoint{0.582cm}{1.696cm}}
\pgfpathcurveto{\pgfqpoint{0.608cm}{1.67cm}}{\pgfqpoint{0.643cm}{1.656cm}}{\pgfqpoint{0.679cm}{1.656cm}}
\pgfpathcurveto{\pgfqpoint{0.715cm}{1.656cm}}{\pgfqpoint{0.75cm}{1.67cm}}{\pgfqpoint{0.775cm}{1.696cm}}
\pgfpathcurveto{\pgfqpoint{0.801cm}{1.722cm}}{\pgfqpoint{0.815cm}{1.756cm}}{\pgfqpoint{0.815cm}{1.793cm}}
\pgfusepath{fill}
\pgfpathmoveto{\pgfqpoint{1.345cm}{1.765cm}}
\pgfpathcurveto{\pgfqpoint{1.345cm}{1.801cm}}{\pgfqpoint{1.331cm}{1.836cm}}{\pgfqpoint{1.305cm}{1.862cm}}
\pgfpathcurveto{\pgfqpoint{1.28cm}{1.887cm}}{\pgfqpoint{1.245cm}{1.902cm}}{\pgfqpoint{1.209cm}{1.902cm}}
\pgfpathcurveto{\pgfqpoint{1.172cm}{1.902cm}}{\pgfqpoint{1.138cm}{1.887cm}}{\pgfqpoint{1.112cm}{1.862cm}}
\pgfpathcurveto{\pgfqpoint{1.087cm}{1.836cm}}{\pgfqpoint{1.072cm}{1.801cm}}{\pgfqpoint{1.072cm}{1.765cm}}
\pgfpathcurveto{\pgfqpoint{1.072cm}{1.728cm}}{\pgfqpoint{1.087cm}{1.694cm}}{\pgfqpoint{1.112cm}{1.668cm}}
\pgfpathcurveto{\pgfqpoint{1.138cm}{1.642cm}}{\pgfqpoint{1.172cm}{1.628cm}}{\pgfqpoint{1.209cm}{1.628cm}}
\pgfpathcurveto{\pgfqpoint{1.245cm}{1.628cm}}{\pgfqpoint{1.28cm}{1.642cm}}{\pgfqpoint{1.305cm}{1.668cm}}
\pgfpathcurveto{\pgfqpoint{1.331cm}{1.694cm}}{\pgfqpoint{1.345cm}{1.728cm}}{\pgfqpoint{1.345cm}{1.765cm}}
\pgfusepath{fill}
\begin{pgfscope}
\pgfsetdash{}{0cm}
\pgfsetlinewidth{0.818mm}
\pgfsetroundcap
\pgfsetroundjoin
\pgfsetmiterlimit{7.0}
\pgfpathmoveto{\pgfqpoint{0.682cm}{1.065cm}}
\pgfpathlineto{\pgfqpoint{1.246cm}{0.315cm}}
\pgfpathlineto{\pgfqpoint{1.811cm}{1.065cm}}
\pgfusepath{stroke}
\end{pgfscope}
\pgfpathmoveto{\pgfqpoint{1.948cm}{1.065cm}}
\pgfpathcurveto{\pgfqpoint{1.948cm}{1.101cm}}{\pgfqpoint{1.933cm}{1.136cm}}{\pgfqpoint{1.907cm}{1.162cm}}
\pgfpathcurveto{\pgfqpoint{1.882cm}{1.187cm}}{\pgfqpoint{1.847cm}{1.202cm}}{\pgfqpoint{1.811cm}{1.202cm}}
\pgfpathcurveto{\pgfqpoint{1.775cm}{1.202cm}}{\pgfqpoint{1.74cm}{1.187cm}}{\pgfqpoint{1.714cm}{1.162cm}}
\pgfpathcurveto{\pgfqpoint{1.689cm}{1.136cm}}{\pgfqpoint{1.674cm}{1.101cm}}{\pgfqpoint{1.674cm}{1.065cm}}
\pgfpathcurveto{\pgfqpoint{1.674cm}{1.029cm}}{\pgfqpoint{1.689cm}{0.994cm}}{\pgfqpoint{1.714cm}{0.968cm}}
\pgfpathcurveto{\pgfqpoint{1.74cm}{0.942cm}}{\pgfqpoint{1.775cm}{0.928cm}}{\pgfqpoint{1.811cm}{0.928cm}}
\pgfpathcurveto{\pgfqpoint{1.847cm}{0.928cm}}{\pgfqpoint{1.882cm}{0.942cm}}{\pgfqpoint{1.907cm}{0.968cm}}
\pgfpathcurveto{\pgfqpoint{1.933cm}{0.994cm}}{\pgfqpoint{1.948cm}{1.029cm}}{\pgfqpoint{1.948cm}{1.065cm}}
\pgfusepath{fill}
\begin{pgfscope}
\pgfsetdash{}{0cm}
\pgfsetlinewidth{0.818mm}
\pgfsetmiterlimit{7.0}
\pgfpathmoveto{\pgfqpoint{1.246cm}{0.315cm}}
\pgfpathlineto{\pgfqpoint{1.244cm}{1.061cm}}
\pgfusepath{stroke}
\end{pgfscope}
\pgfpathmoveto{\pgfqpoint{1.38cm}{1.065cm}}
\pgfpathcurveto{\pgfqpoint{1.38cm}{1.101cm}}{\pgfqpoint{1.366cm}{1.136cm}}{\pgfqpoint{1.34cm}{1.162cm}}
\pgfpathcurveto{\pgfqpoint{1.315cm}{1.187cm}}{\pgfqpoint{1.28cm}{1.202cm}}{\pgfqpoint{1.244cm}{1.202cm}}
\pgfpathcurveto{\pgfqpoint{1.207cm}{1.202cm}}{\pgfqpoint{1.173cm}{1.187cm}}{\pgfqpoint{1.147cm}{1.162cm}}
\pgfpathcurveto{\pgfqpoint{1.121cm}{1.136cm}}{\pgfqpoint{1.107cm}{1.101cm}}{\pgfqpoint{1.107cm}{1.065cm}}
\pgfpathcurveto{\pgfqpoint{1.107cm}{1.029cm}}{\pgfqpoint{1.121cm}{0.994cm}}{\pgfqpoint{1.147cm}{0.968cm}}
\pgfpathcurveto{\pgfqpoint{1.173cm}{0.942cm}}{\pgfqpoint{1.207cm}{0.928cm}}{\pgfqpoint{1.244cm}{0.928cm}}
\pgfpathcurveto{\pgfqpoint{1.28cm}{0.928cm}}{\pgfqpoint{1.315cm}{0.942cm}}{\pgfqpoint{1.34cm}{0.968cm}}
\pgfpathcurveto{\pgfqpoint{1.366cm}{0.994cm}}{\pgfqpoint{1.38cm}{1.029cm}}{\pgfqpoint{1.38cm}{1.065cm}}
\pgfusepath{fill}
\begin{pgfscope}
\pgfsetdash{}{0cm}
\pgfsetlinewidth{0.818mm}
\pgfsetmiterlimit{4.0}
\pgfpathmoveto{\pgfqpoint{1.383cm}{0.178cm}}
\pgfpathcurveto{\pgfqpoint{1.383cm}{0.214cm}}{\pgfqpoint{1.369cm}{0.249cm}}{\pgfqpoint{1.343cm}{0.275cm}}
\pgfpathcurveto{\pgfqpoint{1.317cm}{0.3cm}}{\pgfqpoint{1.283cm}{0.315cm}}{\pgfqpoint{1.246cm}{0.315cm}}
\pgfpathcurveto{\pgfqpoint{1.21cm}{0.315cm}}{\pgfqpoint{1.175cm}{0.3cm}}{\pgfqpoint{1.15cm}{0.275cm}}
\pgfpathcurveto{\pgfqpoint{1.124cm}{0.249cm}}{\pgfqpoint{1.11cm}{0.214cm}}{\pgfqpoint{1.11cm}{0.178cm}}
\pgfpathcurveto{\pgfqpoint{1.11cm}{0.141cm}}{\pgfqpoint{1.124cm}{0.107cm}}{\pgfqpoint{1.15cm}{0.081cm}}
\pgfpathcurveto{\pgfqpoint{1.175cm}{0.055cm}}{\pgfqpoint{1.21cm}{0.041cm}}{\pgfqpoint{1.246cm}{0.041cm}}
\pgfpathcurveto{\pgfqpoint{1.283cm}{0.041cm}}{\pgfqpoint{1.317cm}{0.055cm}}{\pgfqpoint{1.343cm}{0.081cm}}
\pgfpathcurveto{\pgfqpoint{1.369cm}{0.107cm}}{\pgfqpoint{1.383cm}{0.141cm}}{\pgfqpoint{1.383cm}{0.178cm}}
\pgfusepath{stroke}
\end{pgfscope}
\end{pgfscope}
\end{pgfscope}
\end{pgfscope}
\end{tikzpicture}}} = \lim_{\varepsilon\to 0}X^{\!\resizebox{0.6em}{!}{
\begin{tikzpicture}
\pgfpathmoveto{\pgfqpoint{0cm}{-0.035cm}}
\pgfpathlineto{\pgfqpoint{1.376cm}{-0.035cm}}
\pgfpathlineto{\pgfqpoint{1.376cm}{1.552cm}}
\pgfpathlineto{\pgfqpoint{0cm}{1.552cm}}
\pgfpathclose
\pgfusepath{clip}
\begin{pgfscope}
\begin{pgfscope}
\pgfpathmoveto{\pgfqpoint{0cm}{-0.035cm}}
\pgfpathlineto{\pgfqpoint{1.376cm}{-0.035cm}}
\pgfpathlineto{\pgfqpoint{1.376cm}{1.552cm}}
\pgfpathlineto{\pgfqpoint{0cm}{1.552cm}}
\pgfpathclose
\pgfusepath{clip}
\begin{pgfscope}
\begin{pgfscope}
\pgfsetdash{}{0cm}
\pgfsetlinewidth{0.818mm}
\pgfsetroundcap
\pgfsetroundjoin
\pgfsetmiterlimit{7.0}
\definecolor{eps2pgf_color}{gray}{0}\pgfsetstrokecolor{eps2pgf_color}\pgfsetfillcolor{eps2pgf_color}
\pgfpathmoveto{\pgfqpoint{0.117cm}{1.421cm}}
\pgfpathlineto{\pgfqpoint{0.682cm}{0.671cm}}
\pgfpathlineto{\pgfqpoint{1.246cm}{1.421cm}}
\pgfusepath{stroke}
\end{pgfscope}
\definecolor{eps2pgf_color}{gray}{0}\pgfsetstrokecolor{eps2pgf_color}\pgfsetfillcolor{eps2pgf_color}
\pgfpathmoveto{\pgfqpoint{0.273cm}{1.395cm}}
\pgfpathcurveto{\pgfqpoint{0.273cm}{1.432cm}}{\pgfqpoint{0.259cm}{1.467cm}}{\pgfqpoint{0.233cm}{1.492cm}}
\pgfpathcurveto{\pgfqpoint{0.207cm}{1.518cm}}{\pgfqpoint{0.173cm}{1.532cm}}{\pgfqpoint{0.137cm}{1.532cm}}
\pgfpathcurveto{\pgfqpoint{0.1cm}{1.532cm}}{\pgfqpoint{0.066cm}{1.518cm}}{\pgfqpoint{0.04cm}{1.492cm}}
\pgfpathcurveto{\pgfqpoint{0.014cm}{1.467cm}}{\pgfqpoint{0cm}{1.432cm}}{\pgfqpoint{0cm}{1.395cm}}
\pgfpathcurveto{\pgfqpoint{0cm}{1.359cm}}{\pgfqpoint{0.014cm}{1.324cm}}{\pgfqpoint{0.04cm}{1.299cm}}
\pgfpathcurveto{\pgfqpoint{0.066cm}{1.273cm}}{\pgfqpoint{0.1cm}{1.258cm}}{\pgfqpoint{0.137cm}{1.258cm}}
\pgfpathcurveto{\pgfqpoint{0.173cm}{1.258cm}}{\pgfqpoint{0.207cm}{1.273cm}}{\pgfqpoint{0.233cm}{1.299cm}}
\pgfpathcurveto{\pgfqpoint{0.259cm}{1.324cm}}{\pgfqpoint{0.273cm}{1.359cm}}{\pgfqpoint{0.273cm}{1.395cm}}
\pgfusepath{fill}
\begin{pgfscope}
\pgfsetdash{}{0cm}
\pgfsetlinewidth{0.818mm}
\pgfsetmiterlimit{7.0}
\pgfpathmoveto{\pgfqpoint{0.682cm}{0.671cm}}
\pgfpathlineto{\pgfqpoint{0.679cm}{1.418cm}}
\pgfusepath{stroke}
\end{pgfscope}
\pgfpathmoveto{\pgfqpoint{0.815cm}{1.399cm}}
\pgfpathcurveto{\pgfqpoint{0.815cm}{1.435cm}}{\pgfqpoint{0.801cm}{1.47cm}}{\pgfqpoint{0.775cm}{1.496cm}}
\pgfpathcurveto{\pgfqpoint{0.75cm}{1.521cm}}{\pgfqpoint{0.715cm}{1.536cm}}{\pgfqpoint{0.679cm}{1.536cm}}
\pgfpathcurveto{\pgfqpoint{0.643cm}{1.536cm}}{\pgfqpoint{0.608cm}{1.521cm}}{\pgfqpoint{0.582cm}{1.496cm}}
\pgfpathcurveto{\pgfqpoint{0.557cm}{1.47cm}}{\pgfqpoint{0.542cm}{1.435cm}}{\pgfqpoint{0.542cm}{1.399cm}}
\pgfpathcurveto{\pgfqpoint{0.542cm}{1.363cm}}{\pgfqpoint{0.557cm}{1.328cm}}{\pgfqpoint{0.582cm}{1.302cm}}
\pgfpathcurveto{\pgfqpoint{0.608cm}{1.276cm}}{\pgfqpoint{0.643cm}{1.262cm}}{\pgfqpoint{0.679cm}{1.262cm}}
\pgfpathcurveto{\pgfqpoint{0.715cm}{1.262cm}}{\pgfqpoint{0.75cm}{1.276cm}}{\pgfqpoint{0.775cm}{1.302cm}}
\pgfpathcurveto{\pgfqpoint{0.801cm}{1.328cm}}{\pgfqpoint{0.815cm}{1.363cm}}{\pgfqpoint{0.815cm}{1.399cm}}
\pgfusepath{fill}
\pgfpathmoveto{\pgfqpoint{1.345cm}{1.371cm}}
\pgfpathcurveto{\pgfqpoint{1.345cm}{1.408cm}}{\pgfqpoint{1.331cm}{1.442cm}}{\pgfqpoint{1.305cm}{1.468cm}}
\pgfpathcurveto{\pgfqpoint{1.28cm}{1.494cm}}{\pgfqpoint{1.245cm}{1.508cm}}{\pgfqpoint{1.209cm}{1.508cm}}
\pgfpathcurveto{\pgfqpoint{1.172cm}{1.508cm}}{\pgfqpoint{1.138cm}{1.494cm}}{\pgfqpoint{1.112cm}{1.468cm}}
\pgfpathcurveto{\pgfqpoint{1.087cm}{1.442cm}}{\pgfqpoint{1.072cm}{1.408cm}}{\pgfqpoint{1.072cm}{1.371cm}}
\pgfpathcurveto{\pgfqpoint{1.072cm}{1.335cm}}{\pgfqpoint{1.087cm}{1.3cm}}{\pgfqpoint{1.112cm}{1.274cm}}
\pgfpathcurveto{\pgfqpoint{1.138cm}{1.249cm}}{\pgfqpoint{1.172cm}{1.234cm}}{\pgfqpoint{1.209cm}{1.234cm}}
\pgfpathcurveto{\pgfqpoint{1.245cm}{1.234cm}}{\pgfqpoint{1.28cm}{1.249cm}}{\pgfqpoint{1.305cm}{1.274cm}}
\pgfpathcurveto{\pgfqpoint{1.331cm}{1.3cm}}{\pgfqpoint{1.345cm}{1.335cm}}{\pgfqpoint{1.345cm}{1.371cm}}
\pgfusepath{fill}
\begin{pgfscope}
\pgfsetdash{}{0cm}
\pgfsetlinewidth{0.818mm}
\pgfsetroundcap
\pgfsetmiterlimit{4.0}
\pgfpathmoveto{\pgfqpoint{0.682cm}{0.671cm}}
\pgfpathlineto{\pgfqpoint{0.682cm}{0.042cm}}
\pgfusepath{stroke}
\end{pgfscope}
\end{pgfscope}
\end{pgfscope}
\end{pgfscope}
\end{tikzpicture}}}_{\varepsilon} \circ \llbracket X_{\varepsilon}^2 \rrbracket -  b_{\varepsilon}X_{\varepsilon}, \]
   where $b_{\varepsilon}(t) = 3 \mathbbm{E}[(X^{\!\resizebox{0.6em}{!}{
\begin{tikzpicture}
\pgfpathmoveto{\pgfqpoint{0cm}{0cm}}
\pgfpathlineto{\pgfqpoint{1.376cm}{0cm}}
\pgfpathlineto{\pgfqpoint{1.376cm}{1.588cm}}
\pgfpathlineto{\pgfqpoint{0cm}{1.588cm}}
\pgfpathclose
\pgfusepath{clip}
\begin{pgfscope}
\begin{pgfscope}
\pgfpathmoveto{\pgfqpoint{0cm}{0cm}}
\pgfpathlineto{\pgfqpoint{1.376cm}{0cm}}
\pgfpathlineto{\pgfqpoint{1.376cm}{1.588cm}}
\pgfpathlineto{\pgfqpoint{0cm}{1.588cm}}
\pgfpathclose
\pgfusepath{clip}
\begin{pgfscope}
\begin{pgfscope}
\definecolor{eps2pgf_color}{gray}{0.976471}\pgfsetstrokecolor{eps2pgf_color}\pgfsetfillcolor{eps2pgf_color}
\pgfpathmoveto{\pgfqpoint{0cm}{0cm}}
\pgfpathlineto{\pgfqpoint{1.376cm}{0cm}}
\pgfpathlineto{\pgfqpoint{1.376cm}{1.588cm}}
\pgfpathlineto{\pgfqpoint{0cm}{1.588cm}}
\pgfpathclose
\pgfusepath{fill}
\end{pgfscope}
\begin{pgfscope}
\pgfsetdash{}{0cm}
\pgfsetlinewidth{0.818mm}
\pgfsetroundcap
\pgfsetroundjoin
\pgfsetmiterlimit{7.0}
\definecolor{eps2pgf_color}{gray}{0}\pgfsetstrokecolor{eps2pgf_color}\pgfsetfillcolor{eps2pgf_color}
\pgfpathmoveto{\pgfqpoint{0.117cm}{1.476cm}}
\pgfpathlineto{\pgfqpoint{0.682cm}{0.726cm}}
\pgfpathlineto{\pgfqpoint{1.246cm}{1.476cm}}
\pgfusepath{stroke}
\end{pgfscope}
\definecolor{eps2pgf_color}{gray}{0}\pgfsetstrokecolor{eps2pgf_color}\pgfsetfillcolor{eps2pgf_color}
\pgfpathmoveto{\pgfqpoint{0.273cm}{1.451cm}}
\pgfpathcurveto{\pgfqpoint{0.273cm}{1.487cm}}{\pgfqpoint{0.259cm}{1.522cm}}{\pgfqpoint{0.233cm}{1.547cm}}
\pgfpathcurveto{\pgfqpoint{0.207cm}{1.573cm}}{\pgfqpoint{0.173cm}{1.588cm}}{\pgfqpoint{0.137cm}{1.588cm}}
\pgfpathcurveto{\pgfqpoint{0.1cm}{1.588cm}}{\pgfqpoint{0.066cm}{1.573cm}}{\pgfqpoint{0.04cm}{1.547cm}}
\pgfpathcurveto{\pgfqpoint{0.014cm}{1.522cm}}{\pgfqpoint{0cm}{1.487cm}}{\pgfqpoint{0cm}{1.451cm}}
\pgfpathcurveto{\pgfqpoint{0cm}{1.414cm}}{\pgfqpoint{0.014cm}{1.379cm}}{\pgfqpoint{0.04cm}{1.354cm}}
\pgfpathcurveto{\pgfqpoint{0.066cm}{1.328cm}}{\pgfqpoint{0.1cm}{1.314cm}}{\pgfqpoint{0.137cm}{1.314cm}}
\pgfpathcurveto{\pgfqpoint{0.173cm}{1.314cm}}{\pgfqpoint{0.207cm}{1.328cm}}{\pgfqpoint{0.233cm}{1.354cm}}
\pgfpathcurveto{\pgfqpoint{0.259cm}{1.379cm}}{\pgfqpoint{0.273cm}{1.414cm}}{\pgfqpoint{0.273cm}{1.451cm}}
\pgfusepath{fill}
\pgfpathmoveto{\pgfqpoint{1.345cm}{1.426cm}}
\pgfpathcurveto{\pgfqpoint{1.345cm}{1.463cm}}{\pgfqpoint{1.331cm}{1.497cm}}{\pgfqpoint{1.305cm}{1.523cm}}
\pgfpathcurveto{\pgfqpoint{1.28cm}{1.549cm}}{\pgfqpoint{1.245cm}{1.563cm}}{\pgfqpoint{1.209cm}{1.563cm}}
\pgfpathcurveto{\pgfqpoint{1.172cm}{1.563cm}}{\pgfqpoint{1.138cm}{1.549cm}}{\pgfqpoint{1.112cm}{1.523cm}}
\pgfpathcurveto{\pgfqpoint{1.087cm}{1.497cm}}{\pgfqpoint{1.072cm}{1.463cm}}{\pgfqpoint{1.072cm}{1.426cm}}
\pgfpathcurveto{\pgfqpoint{1.072cm}{1.39cm}}{\pgfqpoint{1.087cm}{1.355cm}}{\pgfqpoint{1.112cm}{1.329cm}}
\pgfpathcurveto{\pgfqpoint{1.138cm}{1.304cm}}{\pgfqpoint{1.172cm}{1.289cm}}{\pgfqpoint{1.209cm}{1.289cm}}
\pgfpathcurveto{\pgfqpoint{1.245cm}{1.289cm}}{\pgfqpoint{1.28cm}{1.304cm}}{\pgfqpoint{1.305cm}{1.329cm}}
\pgfpathcurveto{\pgfqpoint{1.331cm}{1.355cm}}{\pgfqpoint{1.345cm}{1.39cm}}{\pgfqpoint{1.345cm}{1.426cm}}
\pgfusepath{fill}
\begin{pgfscope}
\pgfsetdash{}{0cm}
\pgfsetlinewidth{0.818mm}
\pgfsetroundcap
\pgfsetmiterlimit{4.0}
\pgfpathmoveto{\pgfqpoint{0.682cm}{0.726cm}}
\pgfpathlineto{\pgfqpoint{0.682cm}{0.097cm}}
\pgfusepath{stroke}
\end{pgfscope}
\end{pgfscope}
\end{pgfscope}
\end{pgfscope}
\end{tikzpicture}}}_{\varepsilon} \circ \llbracket X_{\varepsilon}^2 \rrbracket )(t,0)]$ stands for a suitable renormalization constant which is $t$ dependent and such that $\sup_{t\ge 0} |b_\varepsilon(t)| \lesssim |\log \varepsilon|$. Moreover, it can be seen that, for each fixed $\varepsilon$, $b_{\varepsilon}$ is smooth and has bounded first derivative.

 \begin{theorem}\label{thm:renorm43}
Let $d=3$. Let $\rho(t,x)=\langle (t,x)\rangle^{-\nu}$ for some $\nu>0$. Then there exist random distributions
\begin{equation}\label{eq:r43}
X,\llbracket X^{2}\rrbracket, X^{\!\resizebox{0.6em}{!}{
\begin{tikzpicture}
\pgfpathmoveto{\pgfqpoint{0cm}{-0.035cm}}
\pgfpathlineto{\pgfqpoint{1.376cm}{-0.035cm}}
\pgfpathlineto{\pgfqpoint{1.376cm}{1.552cm}}
\pgfpathlineto{\pgfqpoint{0cm}{1.552cm}}
\pgfpathclose
\pgfusepath{clip}
\begin{pgfscope}
\begin{pgfscope}
\pgfpathmoveto{\pgfqpoint{0cm}{-0.035cm}}
\pgfpathlineto{\pgfqpoint{1.376cm}{-0.035cm}}
\pgfpathlineto{\pgfqpoint{1.376cm}{1.552cm}}
\pgfpathlineto{\pgfqpoint{0cm}{1.552cm}}
\pgfpathclose
\pgfusepath{clip}
\begin{pgfscope}
\begin{pgfscope}
\pgfsetdash{}{0cm}
\pgfsetlinewidth{0.818mm}
\pgfsetroundcap
\pgfsetroundjoin
\pgfsetmiterlimit{7.0}
\definecolor{eps2pgf_color}{gray}{0}\pgfsetstrokecolor{eps2pgf_color}\pgfsetfillcolor{eps2pgf_color}
\pgfpathmoveto{\pgfqpoint{0.117cm}{1.421cm}}
\pgfpathlineto{\pgfqpoint{0.682cm}{0.671cm}}
\pgfpathlineto{\pgfqpoint{1.246cm}{1.421cm}}
\pgfusepath{stroke}
\end{pgfscope}
\definecolor{eps2pgf_color}{gray}{0}\pgfsetstrokecolor{eps2pgf_color}\pgfsetfillcolor{eps2pgf_color}
\pgfpathmoveto{\pgfqpoint{0.273cm}{1.395cm}}
\pgfpathcurveto{\pgfqpoint{0.273cm}{1.432cm}}{\pgfqpoint{0.259cm}{1.467cm}}{\pgfqpoint{0.233cm}{1.492cm}}
\pgfpathcurveto{\pgfqpoint{0.207cm}{1.518cm}}{\pgfqpoint{0.173cm}{1.532cm}}{\pgfqpoint{0.137cm}{1.532cm}}
\pgfpathcurveto{\pgfqpoint{0.1cm}{1.532cm}}{\pgfqpoint{0.066cm}{1.518cm}}{\pgfqpoint{0.04cm}{1.492cm}}
\pgfpathcurveto{\pgfqpoint{0.014cm}{1.467cm}}{\pgfqpoint{0cm}{1.432cm}}{\pgfqpoint{0cm}{1.395cm}}
\pgfpathcurveto{\pgfqpoint{0cm}{1.359cm}}{\pgfqpoint{0.014cm}{1.324cm}}{\pgfqpoint{0.04cm}{1.299cm}}
\pgfpathcurveto{\pgfqpoint{0.066cm}{1.273cm}}{\pgfqpoint{0.1cm}{1.258cm}}{\pgfqpoint{0.137cm}{1.258cm}}
\pgfpathcurveto{\pgfqpoint{0.173cm}{1.258cm}}{\pgfqpoint{0.207cm}{1.273cm}}{\pgfqpoint{0.233cm}{1.299cm}}
\pgfpathcurveto{\pgfqpoint{0.259cm}{1.324cm}}{\pgfqpoint{0.273cm}{1.359cm}}{\pgfqpoint{0.273cm}{1.395cm}}
\pgfusepath{fill}
\begin{pgfscope}
\pgfsetdash{}{0cm}
\pgfsetlinewidth{0.818mm}
\pgfsetmiterlimit{7.0}
\pgfpathmoveto{\pgfqpoint{0.682cm}{0.671cm}}
\pgfpathlineto{\pgfqpoint{0.679cm}{1.418cm}}
\pgfusepath{stroke}
\end{pgfscope}
\pgfpathmoveto{\pgfqpoint{0.815cm}{1.399cm}}
\pgfpathcurveto{\pgfqpoint{0.815cm}{1.435cm}}{\pgfqpoint{0.801cm}{1.47cm}}{\pgfqpoint{0.775cm}{1.496cm}}
\pgfpathcurveto{\pgfqpoint{0.75cm}{1.521cm}}{\pgfqpoint{0.715cm}{1.536cm}}{\pgfqpoint{0.679cm}{1.536cm}}
\pgfpathcurveto{\pgfqpoint{0.643cm}{1.536cm}}{\pgfqpoint{0.608cm}{1.521cm}}{\pgfqpoint{0.582cm}{1.496cm}}
\pgfpathcurveto{\pgfqpoint{0.557cm}{1.47cm}}{\pgfqpoint{0.542cm}{1.435cm}}{\pgfqpoint{0.542cm}{1.399cm}}
\pgfpathcurveto{\pgfqpoint{0.542cm}{1.363cm}}{\pgfqpoint{0.557cm}{1.328cm}}{\pgfqpoint{0.582cm}{1.302cm}}
\pgfpathcurveto{\pgfqpoint{0.608cm}{1.276cm}}{\pgfqpoint{0.643cm}{1.262cm}}{\pgfqpoint{0.679cm}{1.262cm}}
\pgfpathcurveto{\pgfqpoint{0.715cm}{1.262cm}}{\pgfqpoint{0.75cm}{1.276cm}}{\pgfqpoint{0.775cm}{1.302cm}}
\pgfpathcurveto{\pgfqpoint{0.801cm}{1.328cm}}{\pgfqpoint{0.815cm}{1.363cm}}{\pgfqpoint{0.815cm}{1.399cm}}
\pgfusepath{fill}
\pgfpathmoveto{\pgfqpoint{1.345cm}{1.371cm}}
\pgfpathcurveto{\pgfqpoint{1.345cm}{1.408cm}}{\pgfqpoint{1.331cm}{1.442cm}}{\pgfqpoint{1.305cm}{1.468cm}}
\pgfpathcurveto{\pgfqpoint{1.28cm}{1.494cm}}{\pgfqpoint{1.245cm}{1.508cm}}{\pgfqpoint{1.209cm}{1.508cm}}
\pgfpathcurveto{\pgfqpoint{1.172cm}{1.508cm}}{\pgfqpoint{1.138cm}{1.494cm}}{\pgfqpoint{1.112cm}{1.468cm}}
\pgfpathcurveto{\pgfqpoint{1.087cm}{1.442cm}}{\pgfqpoint{1.072cm}{1.408cm}}{\pgfqpoint{1.072cm}{1.371cm}}
\pgfpathcurveto{\pgfqpoint{1.072cm}{1.335cm}}{\pgfqpoint{1.087cm}{1.3cm}}{\pgfqpoint{1.112cm}{1.274cm}}
\pgfpathcurveto{\pgfqpoint{1.138cm}{1.249cm}}{\pgfqpoint{1.172cm}{1.234cm}}{\pgfqpoint{1.209cm}{1.234cm}}
\pgfpathcurveto{\pgfqpoint{1.245cm}{1.234cm}}{\pgfqpoint{1.28cm}{1.249cm}}{\pgfqpoint{1.305cm}{1.274cm}}
\pgfpathcurveto{\pgfqpoint{1.331cm}{1.3cm}}{\pgfqpoint{1.345cm}{1.335cm}}{\pgfqpoint{1.345cm}{1.371cm}}
\pgfusepath{fill}
\begin{pgfscope}
\pgfsetdash{}{0cm}
\pgfsetlinewidth{0.818mm}
\pgfsetroundcap
\pgfsetmiterlimit{4.0}
\pgfpathmoveto{\pgfqpoint{0.682cm}{0.671cm}}
\pgfpathlineto{\pgfqpoint{0.682cm}{0.042cm}}
\pgfusepath{stroke}
\end{pgfscope}
\end{pgfscope}
\end{pgfscope}
\end{pgfscope}
\end{tikzpicture}}}, X^{\!\resizebox{0.6em}{!}{
\begin{tikzpicture}
\pgfpathmoveto{\pgfqpoint{0cm}{0cm}}
\pgfpathlineto{\pgfqpoint{1.376cm}{0cm}}
\pgfpathlineto{\pgfqpoint{1.376cm}{1.588cm}}
\pgfpathlineto{\pgfqpoint{0cm}{1.588cm}}
\pgfpathclose
\pgfusepath{clip}
\begin{pgfscope}
\begin{pgfscope}
\pgfpathmoveto{\pgfqpoint{0cm}{0cm}}
\pgfpathlineto{\pgfqpoint{1.376cm}{0cm}}
\pgfpathlineto{\pgfqpoint{1.376cm}{1.588cm}}
\pgfpathlineto{\pgfqpoint{0cm}{1.588cm}}
\pgfpathclose
\pgfusepath{clip}
\begin{pgfscope}
\begin{pgfscope}
\definecolor{eps2pgf_color}{gray}{0.976471}\pgfsetstrokecolor{eps2pgf_color}\pgfsetfillcolor{eps2pgf_color}
\pgfpathmoveto{\pgfqpoint{0cm}{0cm}}
\pgfpathlineto{\pgfqpoint{1.376cm}{0cm}}
\pgfpathlineto{\pgfqpoint{1.376cm}{1.588cm}}
\pgfpathlineto{\pgfqpoint{0cm}{1.588cm}}
\pgfpathclose
\pgfusepath{fill}
\end{pgfscope}
\begin{pgfscope}
\pgfsetdash{}{0cm}
\pgfsetlinewidth{0.818mm}
\pgfsetroundcap
\pgfsetroundjoin
\pgfsetmiterlimit{7.0}
\definecolor{eps2pgf_color}{gray}{0}\pgfsetstrokecolor{eps2pgf_color}\pgfsetfillcolor{eps2pgf_color}
\pgfpathmoveto{\pgfqpoint{0.117cm}{1.476cm}}
\pgfpathlineto{\pgfqpoint{0.682cm}{0.726cm}}
\pgfpathlineto{\pgfqpoint{1.246cm}{1.476cm}}
\pgfusepath{stroke}
\end{pgfscope}
\definecolor{eps2pgf_color}{gray}{0}\pgfsetstrokecolor{eps2pgf_color}\pgfsetfillcolor{eps2pgf_color}
\pgfpathmoveto{\pgfqpoint{0.273cm}{1.451cm}}
\pgfpathcurveto{\pgfqpoint{0.273cm}{1.487cm}}{\pgfqpoint{0.259cm}{1.522cm}}{\pgfqpoint{0.233cm}{1.547cm}}
\pgfpathcurveto{\pgfqpoint{0.207cm}{1.573cm}}{\pgfqpoint{0.173cm}{1.588cm}}{\pgfqpoint{0.137cm}{1.588cm}}
\pgfpathcurveto{\pgfqpoint{0.1cm}{1.588cm}}{\pgfqpoint{0.066cm}{1.573cm}}{\pgfqpoint{0.04cm}{1.547cm}}
\pgfpathcurveto{\pgfqpoint{0.014cm}{1.522cm}}{\pgfqpoint{0cm}{1.487cm}}{\pgfqpoint{0cm}{1.451cm}}
\pgfpathcurveto{\pgfqpoint{0cm}{1.414cm}}{\pgfqpoint{0.014cm}{1.379cm}}{\pgfqpoint{0.04cm}{1.354cm}}
\pgfpathcurveto{\pgfqpoint{0.066cm}{1.328cm}}{\pgfqpoint{0.1cm}{1.314cm}}{\pgfqpoint{0.137cm}{1.314cm}}
\pgfpathcurveto{\pgfqpoint{0.173cm}{1.314cm}}{\pgfqpoint{0.207cm}{1.328cm}}{\pgfqpoint{0.233cm}{1.354cm}}
\pgfpathcurveto{\pgfqpoint{0.259cm}{1.379cm}}{\pgfqpoint{0.273cm}{1.414cm}}{\pgfqpoint{0.273cm}{1.451cm}}
\pgfusepath{fill}
\pgfpathmoveto{\pgfqpoint{1.345cm}{1.426cm}}
\pgfpathcurveto{\pgfqpoint{1.345cm}{1.463cm}}{\pgfqpoint{1.331cm}{1.497cm}}{\pgfqpoint{1.305cm}{1.523cm}}
\pgfpathcurveto{\pgfqpoint{1.28cm}{1.549cm}}{\pgfqpoint{1.245cm}{1.563cm}}{\pgfqpoint{1.209cm}{1.563cm}}
\pgfpathcurveto{\pgfqpoint{1.172cm}{1.563cm}}{\pgfqpoint{1.138cm}{1.549cm}}{\pgfqpoint{1.112cm}{1.523cm}}
\pgfpathcurveto{\pgfqpoint{1.087cm}{1.497cm}}{\pgfqpoint{1.072cm}{1.463cm}}{\pgfqpoint{1.072cm}{1.426cm}}
\pgfpathcurveto{\pgfqpoint{1.072cm}{1.39cm}}{\pgfqpoint{1.087cm}{1.355cm}}{\pgfqpoint{1.112cm}{1.329cm}}
\pgfpathcurveto{\pgfqpoint{1.138cm}{1.304cm}}{\pgfqpoint{1.172cm}{1.289cm}}{\pgfqpoint{1.209cm}{1.289cm}}
\pgfpathcurveto{\pgfqpoint{1.245cm}{1.289cm}}{\pgfqpoint{1.28cm}{1.304cm}}{\pgfqpoint{1.305cm}{1.329cm}}
\pgfpathcurveto{\pgfqpoint{1.331cm}{1.355cm}}{\pgfqpoint{1.345cm}{1.39cm}}{\pgfqpoint{1.345cm}{1.426cm}}
\pgfusepath{fill}
\begin{pgfscope}
\pgfsetdash{}{0cm}
\pgfsetlinewidth{0.818mm}
\pgfsetroundcap
\pgfsetmiterlimit{4.0}
\pgfpathmoveto{\pgfqpoint{0.682cm}{0.726cm}}
\pgfpathlineto{\pgfqpoint{0.682cm}{0.097cm}}
\pgfusepath{stroke}
\end{pgfscope}
\end{pgfscope}
\end{pgfscope}
\end{pgfscope}
\end{tikzpicture}}}, X^{\!\resizebox{!}{.8em}{
\begin{tikzpicture}
\pgfpathmoveto{\pgfqpoint{0cm}{-0.035cm}}
\pgfpathlineto{\pgfqpoint{1.976cm}{-0.035cm}}
\pgfpathlineto{\pgfqpoint{1.976cm}{1.94cm}}
\pgfpathlineto{\pgfqpoint{0cm}{1.94cm}}
\pgfpathclose
\pgfusepath{clip}
\begin{pgfscope}
\begin{pgfscope}
\pgfpathmoveto{\pgfqpoint{0cm}{-0.035cm}}
\pgfpathlineto{\pgfqpoint{1.976cm}{-0.035cm}}
\pgfpathlineto{\pgfqpoint{1.976cm}{1.94cm}}
\pgfpathlineto{\pgfqpoint{0cm}{1.94cm}}
\pgfpathclose
\pgfusepath{clip}
\begin{pgfscope}
\begin{pgfscope}
\pgfsetdash{}{0cm}
\pgfsetlinewidth{0.818mm}
\pgfsetroundcap
\pgfsetroundjoin
\pgfsetmiterlimit{7.0}
\definecolor{eps2pgf_color}{gray}{0}\pgfsetstrokecolor{eps2pgf_color}\pgfsetfillcolor{eps2pgf_color}
\pgfpathmoveto{\pgfqpoint{0.117cm}{1.815cm}}
\pgfpathlineto{\pgfqpoint{0.682cm}{1.065cm}}
\pgfpathlineto{\pgfqpoint{1.246cm}{1.815cm}}
\pgfusepath{stroke}
\end{pgfscope}
\definecolor{eps2pgf_color}{gray}{0}\pgfsetstrokecolor{eps2pgf_color}\pgfsetfillcolor{eps2pgf_color}
\pgfpathmoveto{\pgfqpoint{0.273cm}{1.789cm}}
\pgfpathcurveto{\pgfqpoint{0.273cm}{1.825cm}}{\pgfqpoint{0.259cm}{1.86cm}}{\pgfqpoint{0.233cm}{1.886cm}}
\pgfpathcurveto{\pgfqpoint{0.207cm}{1.912cm}}{\pgfqpoint{0.173cm}{1.926cm}}{\pgfqpoint{0.137cm}{1.926cm}}
\pgfpathcurveto{\pgfqpoint{0.1cm}{1.926cm}}{\pgfqpoint{0.066cm}{1.912cm}}{\pgfqpoint{0.04cm}{1.886cm}}
\pgfpathcurveto{\pgfqpoint{0.014cm}{1.86cm}}{\pgfqpoint{0cm}{1.825cm}}{\pgfqpoint{0cm}{1.789cm}}
\pgfpathcurveto{\pgfqpoint{0cm}{1.753cm}}{\pgfqpoint{0.014cm}{1.718cm}}{\pgfqpoint{0.04cm}{1.692cm}}
\pgfpathcurveto{\pgfqpoint{0.066cm}{1.667cm}}{\pgfqpoint{0.1cm}{1.652cm}}{\pgfqpoint{0.137cm}{1.652cm}}
\pgfpathcurveto{\pgfqpoint{0.173cm}{1.652cm}}{\pgfqpoint{0.207cm}{1.667cm}}{\pgfqpoint{0.233cm}{1.692cm}}
\pgfpathcurveto{\pgfqpoint{0.259cm}{1.718cm}}{\pgfqpoint{0.273cm}{1.753cm}}{\pgfqpoint{0.273cm}{1.789cm}}
\pgfusepath{fill}
\begin{pgfscope}
\pgfsetdash{}{0cm}
\pgfsetlinewidth{0.818mm}
\pgfsetmiterlimit{7.0}
\pgfpathmoveto{\pgfqpoint{0.682cm}{1.065cm}}
\pgfpathlineto{\pgfqpoint{0.679cm}{1.812cm}}
\pgfusepath{stroke}
\end{pgfscope}
\pgfpathmoveto{\pgfqpoint{0.815cm}{1.793cm}}
\pgfpathcurveto{\pgfqpoint{0.815cm}{1.829cm}}{\pgfqpoint{0.801cm}{1.864cm}}{\pgfqpoint{0.775cm}{1.89cm}}
\pgfpathcurveto{\pgfqpoint{0.75cm}{1.915cm}}{\pgfqpoint{0.715cm}{1.93cm}}{\pgfqpoint{0.679cm}{1.93cm}}
\pgfpathcurveto{\pgfqpoint{0.643cm}{1.93cm}}{\pgfqpoint{0.608cm}{1.915cm}}{\pgfqpoint{0.582cm}{1.89cm}}
\pgfpathcurveto{\pgfqpoint{0.557cm}{1.864cm}}{\pgfqpoint{0.542cm}{1.829cm}}{\pgfqpoint{0.542cm}{1.793cm}}
\pgfpathcurveto{\pgfqpoint{0.542cm}{1.756cm}}{\pgfqpoint{0.557cm}{1.722cm}}{\pgfqpoint{0.582cm}{1.696cm}}
\pgfpathcurveto{\pgfqpoint{0.608cm}{1.67cm}}{\pgfqpoint{0.643cm}{1.656cm}}{\pgfqpoint{0.679cm}{1.656cm}}
\pgfpathcurveto{\pgfqpoint{0.715cm}{1.656cm}}{\pgfqpoint{0.75cm}{1.67cm}}{\pgfqpoint{0.775cm}{1.696cm}}
\pgfpathcurveto{\pgfqpoint{0.801cm}{1.722cm}}{\pgfqpoint{0.815cm}{1.756cm}}{\pgfqpoint{0.815cm}{1.793cm}}
\pgfusepath{fill}
\pgfpathmoveto{\pgfqpoint{1.345cm}{1.765cm}}
\pgfpathcurveto{\pgfqpoint{1.345cm}{1.801cm}}{\pgfqpoint{1.331cm}{1.836cm}}{\pgfqpoint{1.305cm}{1.862cm}}
\pgfpathcurveto{\pgfqpoint{1.28cm}{1.887cm}}{\pgfqpoint{1.245cm}{1.902cm}}{\pgfqpoint{1.209cm}{1.902cm}}
\pgfpathcurveto{\pgfqpoint{1.172cm}{1.902cm}}{\pgfqpoint{1.138cm}{1.887cm}}{\pgfqpoint{1.112cm}{1.862cm}}
\pgfpathcurveto{\pgfqpoint{1.087cm}{1.836cm}}{\pgfqpoint{1.072cm}{1.801cm}}{\pgfqpoint{1.072cm}{1.765cm}}
\pgfpathcurveto{\pgfqpoint{1.072cm}{1.728cm}}{\pgfqpoint{1.087cm}{1.694cm}}{\pgfqpoint{1.112cm}{1.668cm}}
\pgfpathcurveto{\pgfqpoint{1.138cm}{1.642cm}}{\pgfqpoint{1.172cm}{1.628cm}}{\pgfqpoint{1.209cm}{1.628cm}}
\pgfpathcurveto{\pgfqpoint{1.245cm}{1.628cm}}{\pgfqpoint{1.28cm}{1.642cm}}{\pgfqpoint{1.305cm}{1.668cm}}
\pgfpathcurveto{\pgfqpoint{1.331cm}{1.694cm}}{\pgfqpoint{1.345cm}{1.728cm}}{\pgfqpoint{1.345cm}{1.765cm}}
\pgfusepath{fill}
\begin{pgfscope}
\pgfsetdash{}{0cm}
\pgfsetlinewidth{0.818mm}
\pgfsetroundcap
\pgfsetroundjoin
\pgfsetmiterlimit{7.0}
\pgfpathmoveto{\pgfqpoint{0.682cm}{1.065cm}}
\pgfpathlineto{\pgfqpoint{1.246cm}{0.315cm}}
\pgfpathlineto{\pgfqpoint{1.811cm}{1.065cm}}
\pgfusepath{stroke}
\end{pgfscope}
\pgfpathmoveto{\pgfqpoint{1.948cm}{1.065cm}}
\pgfpathcurveto{\pgfqpoint{1.948cm}{1.101cm}}{\pgfqpoint{1.933cm}{1.136cm}}{\pgfqpoint{1.907cm}{1.162cm}}
\pgfpathcurveto{\pgfqpoint{1.882cm}{1.187cm}}{\pgfqpoint{1.847cm}{1.202cm}}{\pgfqpoint{1.811cm}{1.202cm}}
\pgfpathcurveto{\pgfqpoint{1.775cm}{1.202cm}}{\pgfqpoint{1.74cm}{1.187cm}}{\pgfqpoint{1.714cm}{1.162cm}}
\pgfpathcurveto{\pgfqpoint{1.689cm}{1.136cm}}{\pgfqpoint{1.674cm}{1.101cm}}{\pgfqpoint{1.674cm}{1.065cm}}
\pgfpathcurveto{\pgfqpoint{1.674cm}{1.029cm}}{\pgfqpoint{1.689cm}{0.994cm}}{\pgfqpoint{1.714cm}{0.968cm}}
\pgfpathcurveto{\pgfqpoint{1.74cm}{0.942cm}}{\pgfqpoint{1.775cm}{0.928cm}}{\pgfqpoint{1.811cm}{0.928cm}}
\pgfpathcurveto{\pgfqpoint{1.847cm}{0.928cm}}{\pgfqpoint{1.882cm}{0.942cm}}{\pgfqpoint{1.907cm}{0.968cm}}
\pgfpathcurveto{\pgfqpoint{1.933cm}{0.994cm}}{\pgfqpoint{1.948cm}{1.029cm}}{\pgfqpoint{1.948cm}{1.065cm}}
\pgfusepath{fill}
\begin{pgfscope}
\pgfsetdash{}{0cm}
\pgfsetlinewidth{0.818mm}
\pgfsetmiterlimit{4.0}
\pgfpathmoveto{\pgfqpoint{1.383cm}{0.178cm}}
\pgfpathcurveto{\pgfqpoint{1.383cm}{0.214cm}}{\pgfqpoint{1.369cm}{0.249cm}}{\pgfqpoint{1.343cm}{0.275cm}}
\pgfpathcurveto{\pgfqpoint{1.317cm}{0.3cm}}{\pgfqpoint{1.283cm}{0.315cm}}{\pgfqpoint{1.246cm}{0.315cm}}
\pgfpathcurveto{\pgfqpoint{1.21cm}{0.315cm}}{\pgfqpoint{1.175cm}{0.3cm}}{\pgfqpoint{1.15cm}{0.275cm}}
\pgfpathcurveto{\pgfqpoint{1.124cm}{0.249cm}}{\pgfqpoint{1.11cm}{0.214cm}}{\pgfqpoint{1.11cm}{0.178cm}}
\pgfpathcurveto{\pgfqpoint{1.11cm}{0.141cm}}{\pgfqpoint{1.124cm}{0.107cm}}{\pgfqpoint{1.15cm}{0.081cm}}
\pgfpathcurveto{\pgfqpoint{1.175cm}{0.055cm}}{\pgfqpoint{1.21cm}{0.041cm}}{\pgfqpoint{1.246cm}{0.041cm}}
\pgfpathcurveto{\pgfqpoint{1.283cm}{0.041cm}}{\pgfqpoint{1.317cm}{0.055cm}}{\pgfqpoint{1.343cm}{0.081cm}}
\pgfpathcurveto{\pgfqpoint{1.369cm}{0.107cm}}{\pgfqpoint{1.383cm}{0.141cm}}{\pgfqpoint{1.383cm}{0.178cm}}
\pgfusepath{stroke}
\end{pgfscope}
\end{pgfscope}
\end{pgfscope}
\end{pgfscope}
\end{tikzpicture}}}, X^{\!\resizebox{!}{.8em}{
\begin{tikzpicture}
\pgfpathmoveto{\pgfqpoint{0cm}{-0.035cm}}
\pgfpathlineto{\pgfqpoint{1.976cm}{-0.035cm}}
\pgfpathlineto{\pgfqpoint{1.976cm}{1.94cm}}
\pgfpathlineto{\pgfqpoint{0cm}{1.94cm}}
\pgfpathclose
\pgfusepath{clip}
\begin{pgfscope}
\begin{pgfscope}
\pgfpathmoveto{\pgfqpoint{0cm}{-0.035cm}}
\pgfpathlineto{\pgfqpoint{1.976cm}{-0.035cm}}
\pgfpathlineto{\pgfqpoint{1.976cm}{1.94cm}}
\pgfpathlineto{\pgfqpoint{0cm}{1.94cm}}
\pgfpathclose
\pgfusepath{clip}
\begin{pgfscope}
\begin{pgfscope}
\pgfsetdash{}{0cm}
\pgfsetlinewidth{0.818mm}
\pgfsetroundcap
\pgfsetroundjoin
\pgfsetmiterlimit{7.0}
\definecolor{eps2pgf_color}{gray}{0}\pgfsetstrokecolor{eps2pgf_color}\pgfsetfillcolor{eps2pgf_color}
\pgfpathmoveto{\pgfqpoint{0.117cm}{1.815cm}}
\pgfpathlineto{\pgfqpoint{0.682cm}{1.065cm}}
\pgfpathlineto{\pgfqpoint{1.246cm}{1.815cm}}
\pgfusepath{stroke}
\end{pgfscope}
\definecolor{eps2pgf_color}{gray}{0}\pgfsetstrokecolor{eps2pgf_color}\pgfsetfillcolor{eps2pgf_color}
\pgfpathmoveto{\pgfqpoint{0.273cm}{1.789cm}}
\pgfpathcurveto{\pgfqpoint{0.273cm}{1.825cm}}{\pgfqpoint{0.259cm}{1.86cm}}{\pgfqpoint{0.233cm}{1.886cm}}
\pgfpathcurveto{\pgfqpoint{0.207cm}{1.912cm}}{\pgfqpoint{0.173cm}{1.926cm}}{\pgfqpoint{0.137cm}{1.926cm}}
\pgfpathcurveto{\pgfqpoint{0.1cm}{1.926cm}}{\pgfqpoint{0.066cm}{1.912cm}}{\pgfqpoint{0.04cm}{1.886cm}}
\pgfpathcurveto{\pgfqpoint{0.014cm}{1.86cm}}{\pgfqpoint{0cm}{1.825cm}}{\pgfqpoint{0cm}{1.789cm}}
\pgfpathcurveto{\pgfqpoint{0cm}{1.753cm}}{\pgfqpoint{0.014cm}{1.718cm}}{\pgfqpoint{0.04cm}{1.692cm}}
\pgfpathcurveto{\pgfqpoint{0.066cm}{1.667cm}}{\pgfqpoint{0.1cm}{1.652cm}}{\pgfqpoint{0.137cm}{1.652cm}}
\pgfpathcurveto{\pgfqpoint{0.173cm}{1.652cm}}{\pgfqpoint{0.207cm}{1.667cm}}{\pgfqpoint{0.233cm}{1.692cm}}
\pgfpathcurveto{\pgfqpoint{0.259cm}{1.718cm}}{\pgfqpoint{0.273cm}{1.753cm}}{\pgfqpoint{0.273cm}{1.789cm}}
\pgfusepath{fill}
\pgfpathmoveto{\pgfqpoint{1.345cm}{1.765cm}}
\pgfpathcurveto{\pgfqpoint{1.345cm}{1.801cm}}{\pgfqpoint{1.331cm}{1.836cm}}{\pgfqpoint{1.305cm}{1.862cm}}
\pgfpathcurveto{\pgfqpoint{1.28cm}{1.887cm}}{\pgfqpoint{1.245cm}{1.902cm}}{\pgfqpoint{1.209cm}{1.902cm}}
\pgfpathcurveto{\pgfqpoint{1.172cm}{1.902cm}}{\pgfqpoint{1.138cm}{1.887cm}}{\pgfqpoint{1.112cm}{1.862cm}}
\pgfpathcurveto{\pgfqpoint{1.087cm}{1.836cm}}{\pgfqpoint{1.072cm}{1.801cm}}{\pgfqpoint{1.072cm}{1.765cm}}
\pgfpathcurveto{\pgfqpoint{1.072cm}{1.728cm}}{\pgfqpoint{1.087cm}{1.694cm}}{\pgfqpoint{1.112cm}{1.668cm}}
\pgfpathcurveto{\pgfqpoint{1.138cm}{1.642cm}}{\pgfqpoint{1.172cm}{1.628cm}}{\pgfqpoint{1.209cm}{1.628cm}}
\pgfpathcurveto{\pgfqpoint{1.245cm}{1.628cm}}{\pgfqpoint{1.28cm}{1.642cm}}{\pgfqpoint{1.305cm}{1.668cm}}
\pgfpathcurveto{\pgfqpoint{1.331cm}{1.694cm}}{\pgfqpoint{1.345cm}{1.728cm}}{\pgfqpoint{1.345cm}{1.765cm}}
\pgfusepath{fill}
\begin{pgfscope}
\pgfsetdash{}{0cm}
\pgfsetlinewidth{0.818mm}
\pgfsetroundcap
\pgfsetroundjoin
\pgfsetmiterlimit{7.0}
\pgfpathmoveto{\pgfqpoint{0.682cm}{1.065cm}}
\pgfpathlineto{\pgfqpoint{1.246cm}{0.315cm}}
\pgfpathlineto{\pgfqpoint{1.811cm}{1.065cm}}
\pgfusepath{stroke}
\end{pgfscope}
\pgfpathmoveto{\pgfqpoint{1.948cm}{1.065cm}}
\pgfpathcurveto{\pgfqpoint{1.948cm}{1.101cm}}{\pgfqpoint{1.933cm}{1.136cm}}{\pgfqpoint{1.907cm}{1.162cm}}
\pgfpathcurveto{\pgfqpoint{1.882cm}{1.187cm}}{\pgfqpoint{1.847cm}{1.202cm}}{\pgfqpoint{1.811cm}{1.202cm}}
\pgfpathcurveto{\pgfqpoint{1.775cm}{1.202cm}}{\pgfqpoint{1.74cm}{1.187cm}}{\pgfqpoint{1.714cm}{1.162cm}}
\pgfpathcurveto{\pgfqpoint{1.689cm}{1.136cm}}{\pgfqpoint{1.674cm}{1.101cm}}{\pgfqpoint{1.674cm}{1.065cm}}
\pgfpathcurveto{\pgfqpoint{1.674cm}{1.029cm}}{\pgfqpoint{1.689cm}{0.994cm}}{\pgfqpoint{1.714cm}{0.968cm}}
\pgfpathcurveto{\pgfqpoint{1.74cm}{0.942cm}}{\pgfqpoint{1.775cm}{0.928cm}}{\pgfqpoint{1.811cm}{0.928cm}}
\pgfpathcurveto{\pgfqpoint{1.847cm}{0.928cm}}{\pgfqpoint{1.882cm}{0.942cm}}{\pgfqpoint{1.907cm}{0.968cm}}
\pgfpathcurveto{\pgfqpoint{1.933cm}{0.994cm}}{\pgfqpoint{1.948cm}{1.029cm}}{\pgfqpoint{1.948cm}{1.065cm}}
\pgfusepath{fill}
\begin{pgfscope}
\pgfsetdash{}{0cm}
\pgfsetlinewidth{0.818mm}
\pgfsetmiterlimit{7.0}
\pgfpathmoveto{\pgfqpoint{1.246cm}{0.315cm}}
\pgfpathlineto{\pgfqpoint{1.244cm}{1.061cm}}
\pgfusepath{stroke}
\end{pgfscope}
\pgfpathmoveto{\pgfqpoint{1.38cm}{1.065cm}}
\pgfpathcurveto{\pgfqpoint{1.38cm}{1.101cm}}{\pgfqpoint{1.366cm}{1.136cm}}{\pgfqpoint{1.34cm}{1.162cm}}
\pgfpathcurveto{\pgfqpoint{1.315cm}{1.187cm}}{\pgfqpoint{1.28cm}{1.202cm}}{\pgfqpoint{1.244cm}{1.202cm}}
\pgfpathcurveto{\pgfqpoint{1.207cm}{1.202cm}}{\pgfqpoint{1.173cm}{1.187cm}}{\pgfqpoint{1.147cm}{1.162cm}}
\pgfpathcurveto{\pgfqpoint{1.121cm}{1.136cm}}{\pgfqpoint{1.107cm}{1.101cm}}{\pgfqpoint{1.107cm}{1.065cm}}
\pgfpathcurveto{\pgfqpoint{1.107cm}{1.029cm}}{\pgfqpoint{1.121cm}{0.994cm}}{\pgfqpoint{1.147cm}{0.968cm}}
\pgfpathcurveto{\pgfqpoint{1.173cm}{0.942cm}}{\pgfqpoint{1.207cm}{0.928cm}}{\pgfqpoint{1.244cm}{0.928cm}}
\pgfpathcurveto{\pgfqpoint{1.28cm}{0.928cm}}{\pgfqpoint{1.315cm}{0.942cm}}{\pgfqpoint{1.34cm}{0.968cm}}
\pgfpathcurveto{\pgfqpoint{1.366cm}{0.994cm}}{\pgfqpoint{1.38cm}{1.029cm}}{\pgfqpoint{1.38cm}{1.065cm}}
\pgfusepath{fill}
\begin{pgfscope}
\pgfsetdash{}{0cm}
\pgfsetlinewidth{0.818mm}
\pgfsetmiterlimit{4.0}
\pgfpathmoveto{\pgfqpoint{1.383cm}{0.178cm}}
\pgfpathcurveto{\pgfqpoint{1.383cm}{0.214cm}}{\pgfqpoint{1.369cm}{0.249cm}}{\pgfqpoint{1.343cm}{0.275cm}}
\pgfpathcurveto{\pgfqpoint{1.317cm}{0.3cm}}{\pgfqpoint{1.283cm}{0.315cm}}{\pgfqpoint{1.246cm}{0.315cm}}
\pgfpathcurveto{\pgfqpoint{1.21cm}{0.315cm}}{\pgfqpoint{1.175cm}{0.3cm}}{\pgfqpoint{1.15cm}{0.275cm}}
\pgfpathcurveto{\pgfqpoint{1.124cm}{0.249cm}}{\pgfqpoint{1.11cm}{0.214cm}}{\pgfqpoint{1.11cm}{0.178cm}}
\pgfpathcurveto{\pgfqpoint{1.11cm}{0.141cm}}{\pgfqpoint{1.124cm}{0.107cm}}{\pgfqpoint{1.15cm}{0.081cm}}
\pgfpathcurveto{\pgfqpoint{1.175cm}{0.055cm}}{\pgfqpoint{1.21cm}{0.041cm}}{\pgfqpoint{1.246cm}{0.041cm}}
\pgfpathcurveto{\pgfqpoint{1.283cm}{0.041cm}}{\pgfqpoint{1.317cm}{0.055cm}}{\pgfqpoint{1.343cm}{0.081cm}}
\pgfpathcurveto{\pgfqpoint{1.369cm}{0.107cm}}{\pgfqpoint{1.383cm}{0.141cm}}{\pgfqpoint{1.383cm}{0.178cm}}
\pgfusepath{stroke}
\end{pgfscope}
\end{pgfscope}
\end{pgfscope}
\end{pgfscope}
\end{tikzpicture}}}, X^{\!\resizebox{!}{.8em}{
\begin{tikzpicture}
\pgfpathmoveto{\pgfqpoint{0cm}{-0.035cm}}
\pgfpathlineto{\pgfqpoint{1.976cm}{-0.035cm}}
\pgfpathlineto{\pgfqpoint{1.976cm}{1.94cm}}
\pgfpathlineto{\pgfqpoint{0cm}{1.94cm}}
\pgfpathclose
\pgfusepath{clip}
\begin{pgfscope}
\begin{pgfscope}
\pgfpathmoveto{\pgfqpoint{0cm}{-0.035cm}}
\pgfpathlineto{\pgfqpoint{1.976cm}{-0.035cm}}
\pgfpathlineto{\pgfqpoint{1.976cm}{1.94cm}}
\pgfpathlineto{\pgfqpoint{0cm}{1.94cm}}
\pgfpathclose
\pgfusepath{clip}
\begin{pgfscope}
\begin{pgfscope}
\pgfsetdash{}{0cm}
\pgfsetlinewidth{0.818mm}
\pgfsetroundcap
\pgfsetroundjoin
\pgfsetmiterlimit{7.0}
\definecolor{eps2pgf_color}{gray}{0}\pgfsetstrokecolor{eps2pgf_color}\pgfsetfillcolor{eps2pgf_color}
\pgfpathmoveto{\pgfqpoint{0.117cm}{1.815cm}}
\pgfpathlineto{\pgfqpoint{0.682cm}{1.065cm}}
\pgfpathlineto{\pgfqpoint{1.246cm}{1.815cm}}
\pgfusepath{stroke}
\end{pgfscope}
\definecolor{eps2pgf_color}{gray}{0}\pgfsetstrokecolor{eps2pgf_color}\pgfsetfillcolor{eps2pgf_color}
\pgfpathmoveto{\pgfqpoint{0.273cm}{1.789cm}}
\pgfpathcurveto{\pgfqpoint{0.273cm}{1.825cm}}{\pgfqpoint{0.259cm}{1.86cm}}{\pgfqpoint{0.233cm}{1.886cm}}
\pgfpathcurveto{\pgfqpoint{0.207cm}{1.912cm}}{\pgfqpoint{0.173cm}{1.926cm}}{\pgfqpoint{0.137cm}{1.926cm}}
\pgfpathcurveto{\pgfqpoint{0.1cm}{1.926cm}}{\pgfqpoint{0.066cm}{1.912cm}}{\pgfqpoint{0.04cm}{1.886cm}}
\pgfpathcurveto{\pgfqpoint{0.014cm}{1.86cm}}{\pgfqpoint{0cm}{1.825cm}}{\pgfqpoint{0cm}{1.789cm}}
\pgfpathcurveto{\pgfqpoint{0cm}{1.753cm}}{\pgfqpoint{0.014cm}{1.718cm}}{\pgfqpoint{0.04cm}{1.692cm}}
\pgfpathcurveto{\pgfqpoint{0.066cm}{1.667cm}}{\pgfqpoint{0.1cm}{1.652cm}}{\pgfqpoint{0.137cm}{1.652cm}}
\pgfpathcurveto{\pgfqpoint{0.173cm}{1.652cm}}{\pgfqpoint{0.207cm}{1.667cm}}{\pgfqpoint{0.233cm}{1.692cm}}
\pgfpathcurveto{\pgfqpoint{0.259cm}{1.718cm}}{\pgfqpoint{0.273cm}{1.753cm}}{\pgfqpoint{0.273cm}{1.789cm}}
\pgfusepath{fill}
\begin{pgfscope}
\pgfsetdash{}{0cm}
\pgfsetlinewidth{0.818mm}
\pgfsetmiterlimit{7.0}
\pgfpathmoveto{\pgfqpoint{0.682cm}{1.065cm}}
\pgfpathlineto{\pgfqpoint{0.679cm}{1.812cm}}
\pgfusepath{stroke}
\end{pgfscope}
\pgfpathmoveto{\pgfqpoint{0.815cm}{1.793cm}}
\pgfpathcurveto{\pgfqpoint{0.815cm}{1.829cm}}{\pgfqpoint{0.801cm}{1.864cm}}{\pgfqpoint{0.775cm}{1.89cm}}
\pgfpathcurveto{\pgfqpoint{0.75cm}{1.915cm}}{\pgfqpoint{0.715cm}{1.93cm}}{\pgfqpoint{0.679cm}{1.93cm}}
\pgfpathcurveto{\pgfqpoint{0.643cm}{1.93cm}}{\pgfqpoint{0.608cm}{1.915cm}}{\pgfqpoint{0.582cm}{1.89cm}}
\pgfpathcurveto{\pgfqpoint{0.557cm}{1.864cm}}{\pgfqpoint{0.542cm}{1.829cm}}{\pgfqpoint{0.542cm}{1.793cm}}
\pgfpathcurveto{\pgfqpoint{0.542cm}{1.756cm}}{\pgfqpoint{0.557cm}{1.722cm}}{\pgfqpoint{0.582cm}{1.696cm}}
\pgfpathcurveto{\pgfqpoint{0.608cm}{1.67cm}}{\pgfqpoint{0.643cm}{1.656cm}}{\pgfqpoint{0.679cm}{1.656cm}}
\pgfpathcurveto{\pgfqpoint{0.715cm}{1.656cm}}{\pgfqpoint{0.75cm}{1.67cm}}{\pgfqpoint{0.775cm}{1.696cm}}
\pgfpathcurveto{\pgfqpoint{0.801cm}{1.722cm}}{\pgfqpoint{0.815cm}{1.756cm}}{\pgfqpoint{0.815cm}{1.793cm}}
\pgfusepath{fill}
\pgfpathmoveto{\pgfqpoint{1.345cm}{1.765cm}}
\pgfpathcurveto{\pgfqpoint{1.345cm}{1.801cm}}{\pgfqpoint{1.331cm}{1.836cm}}{\pgfqpoint{1.305cm}{1.862cm}}
\pgfpathcurveto{\pgfqpoint{1.28cm}{1.887cm}}{\pgfqpoint{1.245cm}{1.902cm}}{\pgfqpoint{1.209cm}{1.902cm}}
\pgfpathcurveto{\pgfqpoint{1.172cm}{1.902cm}}{\pgfqpoint{1.138cm}{1.887cm}}{\pgfqpoint{1.112cm}{1.862cm}}
\pgfpathcurveto{\pgfqpoint{1.087cm}{1.836cm}}{\pgfqpoint{1.072cm}{1.801cm}}{\pgfqpoint{1.072cm}{1.765cm}}
\pgfpathcurveto{\pgfqpoint{1.072cm}{1.728cm}}{\pgfqpoint{1.087cm}{1.694cm}}{\pgfqpoint{1.112cm}{1.668cm}}
\pgfpathcurveto{\pgfqpoint{1.138cm}{1.642cm}}{\pgfqpoint{1.172cm}{1.628cm}}{\pgfqpoint{1.209cm}{1.628cm}}
\pgfpathcurveto{\pgfqpoint{1.245cm}{1.628cm}}{\pgfqpoint{1.28cm}{1.642cm}}{\pgfqpoint{1.305cm}{1.668cm}}
\pgfpathcurveto{\pgfqpoint{1.331cm}{1.694cm}}{\pgfqpoint{1.345cm}{1.728cm}}{\pgfqpoint{1.345cm}{1.765cm}}
\pgfusepath{fill}
\begin{pgfscope}
\pgfsetdash{}{0cm}
\pgfsetlinewidth{0.818mm}
\pgfsetroundcap
\pgfsetroundjoin
\pgfsetmiterlimit{7.0}
\pgfpathmoveto{\pgfqpoint{0.682cm}{1.065cm}}
\pgfpathlineto{\pgfqpoint{1.246cm}{0.315cm}}
\pgfpathlineto{\pgfqpoint{1.811cm}{1.065cm}}
\pgfusepath{stroke}
\end{pgfscope}
\pgfpathmoveto{\pgfqpoint{1.948cm}{1.065cm}}
\pgfpathcurveto{\pgfqpoint{1.948cm}{1.101cm}}{\pgfqpoint{1.933cm}{1.136cm}}{\pgfqpoint{1.907cm}{1.162cm}}
\pgfpathcurveto{\pgfqpoint{1.882cm}{1.187cm}}{\pgfqpoint{1.847cm}{1.202cm}}{\pgfqpoint{1.811cm}{1.202cm}}
\pgfpathcurveto{\pgfqpoint{1.775cm}{1.202cm}}{\pgfqpoint{1.74cm}{1.187cm}}{\pgfqpoint{1.714cm}{1.162cm}}
\pgfpathcurveto{\pgfqpoint{1.689cm}{1.136cm}}{\pgfqpoint{1.674cm}{1.101cm}}{\pgfqpoint{1.674cm}{1.065cm}}
\pgfpathcurveto{\pgfqpoint{1.674cm}{1.029cm}}{\pgfqpoint{1.689cm}{0.994cm}}{\pgfqpoint{1.714cm}{0.968cm}}
\pgfpathcurveto{\pgfqpoint{1.74cm}{0.942cm}}{\pgfqpoint{1.775cm}{0.928cm}}{\pgfqpoint{1.811cm}{0.928cm}}
\pgfpathcurveto{\pgfqpoint{1.847cm}{0.928cm}}{\pgfqpoint{1.882cm}{0.942cm}}{\pgfqpoint{1.907cm}{0.968cm}}
\pgfpathcurveto{\pgfqpoint{1.933cm}{0.994cm}}{\pgfqpoint{1.948cm}{1.029cm}}{\pgfqpoint{1.948cm}{1.065cm}}
\pgfusepath{fill}
\begin{pgfscope}
\pgfsetdash{}{0cm}
\pgfsetlinewidth{0.818mm}
\pgfsetmiterlimit{7.0}
\pgfpathmoveto{\pgfqpoint{1.246cm}{0.315cm}}
\pgfpathlineto{\pgfqpoint{1.244cm}{1.061cm}}
\pgfusepath{stroke}
\end{pgfscope}
\pgfpathmoveto{\pgfqpoint{1.38cm}{1.065cm}}
\pgfpathcurveto{\pgfqpoint{1.38cm}{1.101cm}}{\pgfqpoint{1.366cm}{1.136cm}}{\pgfqpoint{1.34cm}{1.162cm}}
\pgfpathcurveto{\pgfqpoint{1.315cm}{1.187cm}}{\pgfqpoint{1.28cm}{1.202cm}}{\pgfqpoint{1.244cm}{1.202cm}}
\pgfpathcurveto{\pgfqpoint{1.207cm}{1.202cm}}{\pgfqpoint{1.173cm}{1.187cm}}{\pgfqpoint{1.147cm}{1.162cm}}
\pgfpathcurveto{\pgfqpoint{1.121cm}{1.136cm}}{\pgfqpoint{1.107cm}{1.101cm}}{\pgfqpoint{1.107cm}{1.065cm}}
\pgfpathcurveto{\pgfqpoint{1.107cm}{1.029cm}}{\pgfqpoint{1.121cm}{0.994cm}}{\pgfqpoint{1.147cm}{0.968cm}}
\pgfpathcurveto{\pgfqpoint{1.173cm}{0.942cm}}{\pgfqpoint{1.207cm}{0.928cm}}{\pgfqpoint{1.244cm}{0.928cm}}
\pgfpathcurveto{\pgfqpoint{1.28cm}{0.928cm}}{\pgfqpoint{1.315cm}{0.942cm}}{\pgfqpoint{1.34cm}{0.968cm}}
\pgfpathcurveto{\pgfqpoint{1.366cm}{0.994cm}}{\pgfqpoint{1.38cm}{1.029cm}}{\pgfqpoint{1.38cm}{1.065cm}}
\pgfusepath{fill}
\begin{pgfscope}
\pgfsetdash{}{0cm}
\pgfsetlinewidth{0.818mm}
\pgfsetmiterlimit{4.0}
\pgfpathmoveto{\pgfqpoint{1.383cm}{0.178cm}}
\pgfpathcurveto{\pgfqpoint{1.383cm}{0.214cm}}{\pgfqpoint{1.369cm}{0.249cm}}{\pgfqpoint{1.343cm}{0.275cm}}
\pgfpathcurveto{\pgfqpoint{1.317cm}{0.3cm}}{\pgfqpoint{1.283cm}{0.315cm}}{\pgfqpoint{1.246cm}{0.315cm}}
\pgfpathcurveto{\pgfqpoint{1.21cm}{0.315cm}}{\pgfqpoint{1.175cm}{0.3cm}}{\pgfqpoint{1.15cm}{0.275cm}}
\pgfpathcurveto{\pgfqpoint{1.124cm}{0.249cm}}{\pgfqpoint{1.11cm}{0.214cm}}{\pgfqpoint{1.11cm}{0.178cm}}
\pgfpathcurveto{\pgfqpoint{1.11cm}{0.141cm}}{\pgfqpoint{1.124cm}{0.107cm}}{\pgfqpoint{1.15cm}{0.081cm}}
\pgfpathcurveto{\pgfqpoint{1.175cm}{0.055cm}}{\pgfqpoint{1.21cm}{0.041cm}}{\pgfqpoint{1.246cm}{0.041cm}}
\pgfpathcurveto{\pgfqpoint{1.283cm}{0.041cm}}{\pgfqpoint{1.317cm}{0.055cm}}{\pgfqpoint{1.343cm}{0.081cm}}
\pgfpathcurveto{\pgfqpoint{1.369cm}{0.107cm}}{\pgfqpoint{1.383cm}{0.141cm}}{\pgfqpoint{1.383cm}{0.178cm}}
\pgfusepath{stroke}
\end{pgfscope}
\end{pgfscope}
\end{pgfscope}
\end{pgfscope}
\end{tikzpicture}}}
\end{equation}
such that if $\tau$ denotes one of the distributions in \eqref{eq:r43} then $$\tau\in C\CC^{\alpha_{\tau}}(\rho^{\sigma})\cap C^{\delta/2}\CC^{\alpha_{\tau}-\gamma}(\rho^{\sigma})  $$ for $\alpha_{\tau}$  given by Table \ref{t:reg}, every $\kappa,\sigma>0$ and some $\delta,\gamma>0$. Moreover, if $\tau_{\varepsilon}$ is the smooth version of  $\tau$ then
$\tau_{\varepsilon }\to \tau$ in $ C\CC^{\alpha_{\tau}}(\rho^{\sigma})\cap C^{\delta/2}\CC^{\alpha_{\tau}-\gamma}(\rho^{\sigma})$ a.s. as $\varepsilon\to0$.
 \end{theorem}

 \begin{proof} The convergence and renormalization of the stochastic terms has been performed several times in the literature, see the proof of Theorem~\ref{thm:renorm2} for precise references. As for the convergence in the space-time weighted Besov--H\"older spaces arguments similar to those described in Theorem~\ref{thm:renorm42} can be applied to establish the claim.
 \end{proof}
 
\rmbb{\begin{remark}
We note that $\llbracket X^{3}\rrbracket$ can be only realized as a space-time random distribution and point evaluation for fixed times is not well defined. Thus, $\llbracket X^{3}\rrbracket$ was not included in the statement of Theorem \ref{thm:renorm43}. However, it is not needed in the subsequent analysis of the  $d=3$ parabolic case.
 \end{remark}}

\section{Elliptic $\Phi^4_4$ model}
\label{sec:44}

The goal of this section is  threefold. First, we derive a suitable decomposition of the elliptic $\Phi^{4}$ model \eqref{eq:phi4} in dimension 4. Second, we establish a priori estimates for the involved quantities. This will also serve as a basis for the investigation of the parabolic $\Phi^{4}$ model in dimension 2, see Section \ref{sec:42}. Finally, we employ Schaefer's fixed point theorem together with compactness arguments in order to construct solutions to the decomposed elliptic system.

\subsection{Decomposition into simpler equations}
\label{ssec:decomp}

We  study the elliptic equation
\begin{equation*}
(- \Delta + \mu) \varphi + \varphi^3 - 3 a \varphi - \xi = 0
\end{equation*}
in $\mathbbm{R}^4$ where $\xi$ is a space white noise and $a$ stands for a renormalization constant needed to define the stochastic objects below. We let $(- \Delta +
\mu) = \Q$ and introduce the ansatz
\[ \varphi = X + \phi + \psi, \]
with
\begin{equation*}
 \Q X = \xi, \qquad \llbracket X^3 \rrbracket \assign X^3 - 3 a X, \quad
   \llbracket X^2 \rrbracket \assign X^2 - a.
\end{equation*}
Consequently,
\[ 0 = \Q \varphi + \varphi^3 - 3 a \varphi - \xi = \Q \phi + \Q \psi +
   \llbracket X^3 \rrbracket + 3 (\phi + \psi) \llbracket X^2 \rrbracket + 3
   (\phi + \psi)^2 X + (\phi + \psi)^3 . \]
This equation will be decomposed into a system of equations, namely,
\begin{equation}\label{eq:44b}
\Q \phi + \Phi = 0, \qquad \Q \psi + \psi^3 + \Psi = 0,
\end{equation}
where in $\Phi$ we collect all the contributions of negative regularity and in
$\Psi$ all the others (belonging locally to $L^{\infty}$). In addition, with the help of  the operators $\UU_{\leqslant},
\UU_{>}$ defined in Section~\ref{ssec:local}, we localize all the irregular contributions. Namely, each irregular term depending on $\phi+\psi$ will be decomposed into two parts: one even more irregular but controlled by the $L^{\8}$-norm of $\phi+\psi$; and its regular counterpart, which will be included into $\Psi$. This step will be beneficial for the a priori estimates in Section \ref{ssec:apr} as it allows to estimate $\phi$ easily and therefore  eliminate various norms of $\phi$ from the estimates of $\psi$. In other words, thanks to the localizers $\UU_{\leqslant}, \UU_{>}$ we are able to decouple \eqref{eq:44b} and develop an efficient approach towards a priori estimates.

\rmbb{To be more precise,  recall that the operators $\UU_{\leqslant}, \UU_{>}$ depend on a given parameter $L>0$, which has to be chosen appropriately. Moreover, we will choose different values of $L$ for different stochastic objects while keeping in mind that $\UU_{>}$ and $\UU_{\leq}$ of one object shall be given by the same parameter $L$ in order to maintain $\UU_{>}
+\UU_{\leqslant} = \mathrm{Id}$. For the moment, we keep these parameters fixed but arbitrary and their precise values will be determined  below in Section \ref{ssec:apr}.}

Including the localizers, we define
\begin{align}\label{eq:44c}
\begin{aligned}
\Phi &:= \llbracket X^3 \rrbracket + 3 (\phi + \psi) \prec \UU_{>}
   \llbracket X^2 \rrbracket + 3 (\phi + \psi)^2 \prec \UU_{>} X,\\
    \Psi &:= \Psi_1 + \Psi_2, \\
     \Psi_1 &:= \phi^3 + 3 \psi \phi^2 + 3 \psi^2 \phi,\\
     \Psi_2 &:= 3 (\phi + \psi) \prec \UU_{\leqslant} \llbracket X^2
   \rrbracket + 3 (\phi + \psi)^2 \prec \UU_{\leqslant} X + 3 (\phi +
   \psi) \succcurlyeq \llbracket X^2 \rrbracket + 3 (\phi + \psi)^2
   \succcurlyeq X.
\end{aligned}
\end{align}   
 
\subsection{A priori estimates}
\label{ssec:apr}
 
 \rmbb{Let us fix a constant $K>0$ to be chosen  after \eqref{eq:44e} based on the $L^{\infty}$-norm of $\phi+\psi$. Given this value of $K$, we  now determine the values of $L$ in the localization of $X$ and $\llbracket X^{2}\rrbracket$ appearing in \eqref{eq:44c}.}
To this end, recall that the stochastic objects can be constructed so that
$$
\|X\|_{\CC^{-\kappa}(\rho^\sigma)},\|\llbracket X^2 \rrbracket\|_{\CC^{-\kappa}(\rho^\sigma)},\|\llbracket X^3 \rrbracket\|_{\CC^{-\kappa}(\rho^\sigma)}\lesssim 1
$$
provided $\rho$ is a polynomial weight of the form $\rho(x)=\langle x\rangle^{-\nu}$ for some $\nu>0$ and $\kappa,\sigma>0$. 
Hence in view of Lemma~\ref{lem:local} we can choose small parameters $\alpha>\kappa>0$ and $\delta=2-\kappa-\alpha>0$, $\beta=\alpha-\kappa>0$ to set up the localization operators so that (in the sequel, the parameter $\sigma$ is always positive but may change from bound to bound)
$$
\| \UU_{>} X \|_{\CC^{\alpha-2} (\rho^{- 1})} \lesssim 2^{- \delta K}\|X\|_{\CC^{-\kappa}(\rho^\sigma)},
   \qquad \| \UU_{\leqslant} X \|_{\CC^{\alpha} (\rho^{1})}
   \lesssim 2^{(\alpha+\kappa) K}\|X\|_{\CC^{-\kappa}(\rho^\sigma)},
$$
and
\[ \| \UU_{>} \llbracket X^2 \rrbracket \|_{\CC^{\alpha-2}} \lesssim
   2^{- \delta K/2}\|\llbracket X^2 \rrbracket\|_{\CC^{-\kappa}(\rho^\sigma)}, \qquad \| \UU_{\leqslant} \llbracket X^2 \rrbracket
   \|_{\CC^{\alpha} (\rho^{2 })} \lesssim 2^{(\alpha+\kappa) K/2}\|\llbracket X^2 \rrbracket\|_{\CC^{-\kappa}(\rho^\sigma)}. \]
\rmbb{Remark that we chose different values of the parameter $L$ in the localization of $X$ and $\llbracket X^{2}\rrbracket$, namely, $L=K$ for $X$ and $L=K/2$ for $\llbracket X^{2}\rrbracket$.}
From this we have
$$
\| \Phi \|_{\CC^{\alpha - 2} (\rho)} \lesssim \|\llbracket X^3 \rrbracket\|_{\CC^{-\kappa}(\rho^\sigma)} + \|\phi+\psi\|_{L^\infty(\rho)}\| \UU_{>} \llbracket X^2 \rrbracket \|_{\CC^{\alpha - 2}}+\|\phi+\psi\|_{L^\infty(\rho)}^2\| \UU_{>} X \|_{\CC^{\alpha - 2} (\rho^{- 1})}
$$
$$
\lesssim 1+ \|\phi+\psi\|_{L^\infty(\rho)} 2^{-(2-\kappa-\alpha)K/2}+\|\phi+\psi\|_{L^\infty(\rho)}^2 2^{-(2-\kappa-\alpha)K}.
$$
Hence it follows from the Schauder estimates that
\begin{equation}\label{eq:44e}
\|\phi\|_{\CC^\alpha(\rho)}\lesssim \| \Phi \|_{\CC^{\alpha - 2} (\rho)} \lesssim 1+  \|\phi+\psi\|_{L^\infty(\rho)} 2^{-(2-\kappa-\alpha)K/2}+\|\phi+\psi\|_{L^\infty(\rho)}^2 2^{-(2-\kappa-\alpha)K}.
\end{equation}

\rmbb{This leads us to the precise choice of the parameter $K$. In particular, we recall that since the equation for $\phi+\psi$ does not contain any localizers, the norm $\|\phi+\psi\|_{L^{\infty} (\rho)}$ does not depend on the particular choice of the localizers and is finite by assumption.
Let $K>0$ be such that $1+\|\phi+\psi\|_{L^\infty(\rho)}= 2^{(2-\kappa-\alpha)K/2}$. Then  in view of the embedding \eqref{eq:emb} we deduce from \eqref{eq:44e}
\begin{equation*}
\| \phi \|_{L^{\infty} (\rho)} + \| \phi \|_{\CC^{\alpha} (\rho^{1 +
   \alpha})}+ \| \phi \|_{\CC^{\beta} (\rho^{1 +
   \beta})} \lesssim 1
   \end{equation*}
  where the constant on the right hand side depends on the noise terms but  is independent of $ K$.
  Consequently,
  $$2^{(2-\kappa-\alpha)K/2}\leq 1+\|\phi\|_{L^\infty(\rho)}+\|\psi\|_{L^\infty(\rho)}\lesssim 1+\|\psi\|_{L^\infty(\rho)}$$
 and hence $2^{(\alpha+\kappa)K}\lesssim 1+ \|\psi\|_{L^\infty(\rho)}^{\varepsilon} $ for some $\varepsilon\in (0,1)$ independent of $K$. The implicit constant (here and in the sequel) is also independent of $K$. The parameter $K$ remains fixed for the rest of the analysis of the elliptic $\Phi^{4}_{4}$ model.} 

To proceed with our a priori estimate, we have
\begin{align}\label{eq:psi1}
\begin{aligned}
\| \Psi_1 \|_{\CC^{\beta} (\rho^{3 + \beta})}& \lesssim (1+\|\psi\|_{L^\infty(\rho)})(1+\|\psi\|_{\CC^{\beta}(\rho^{1+\beta})})\\
\| \Psi_1 \|_{L^{\infty} (\rho^3)} &\lesssim  1+
   \| \psi \|^2_{L^{\infty}
   (\rho)}
\end{aligned}
\end{align}
and
   \[ \| \Psi_2 \|_{\CC^{\beta} (\rho^{3 + \beta})} \lesssim (1+\|\psi\|_{L^\infty(\rho)} ) \|\UU_{\leq}\llbracket X^2\rrbracket\|_{\CC^{\beta}(\rho^{2+\beta})}+(1+\|\psi\|^2_{L^\infty(\rho)} ) \|\UU_{\leq} X\|_{\CC^{\beta}(\rho^{1+\beta})}\]
   \[+(1+\|\psi\|_{\CC^\alpha(\rho^{1+\alpha})} ) \|\llbracket X^2\rrbracket\|_{\CC^{-\kappa}(\rho^{2-\kappa})}+(1+\|\psi\|_{\CC^{\alpha}(\rho^{1+\alpha})} )(1+\|\psi\|_{L^\infty(\rho)}) \| X\|_{\CC^{-\kappa}(\rho^{1-\kappa})}\]
   \begin{equation}\label{eq:psi2}
   \lesssim 1+ \|\psi\|_{L^\infty(\rho)}^{2+\varepsilon}+\|\psi\|_{\CC^\alpha(\rho^{1+\alpha})}(1+\|\psi\|_{L^\infty(\rho)}),
   \end{equation}
   which implies also
\begin{equation}\label{eq:psi22}
\| \Psi_2 \|_{L^{\infty} (\rho^3)} \lesssim 1+\|\psi\|_{L^\infty(\rho)}^{2+\varepsilon}+\|\psi\|_{\CC^\alpha(\rho^{1+\alpha})}(1+\|\psi\|_{L^\infty(\rho)}). 
\end{equation}
Thus, according to Lemma \ref{lemma:schauder-ellptic}
\[
\|\psi\|_{\CC^{2+\beta}(\rho^{3+\beta})}\lesssim 1+\|\psi\|_{\CC^\alpha(\rho^{1+\alpha})}(1+\|\psi\|_{L^\infty(\rho)})+\|\psi\|_{L^\infty(\rho)}^{3+\beta} .\]
 Since we have due to Lemma \ref{lemma:interp}
 \[
\|\psi\|_{\CC^{\alpha}(\rho^{1+\alpha})} \lesssim \|\psi\|_{L^\infty(\rho)}^{1-\alpha/(2+\beta)}\|\psi\|^{\alpha/(2+\beta)}_{\CC^{2+\beta}(\rho^{3+\beta})},
\]
the weighted Young inequality yields
\[
\|\psi\|_{\CC^{2+\beta}(\rho^{3+\beta})} \lesssim 1+\|\psi\|_{L^\infty(\rho)}^{3+\beta}.
\]
Now, it holds
 \[ 
 \|\Psi\|_{L^\infty(\rho^3)}\lesssim 1+\|\psi\|_{L^\infty(\rho)}^{2+\varepsilon}
+\|\psi\|_{\CC^\alpha(\rho^{1+\alpha})} (1 + \|\psi\|_{L^\infty(\rho)}) \]
\[
\lesssim 1+\|\psi\|_{L^\infty(\rho)}^{2+\varepsilon}
+\|\psi\|_{L^\infty(\rho)}^{1-\alpha/(2+\beta)}\|\psi\|^{\alpha/(2+\beta)}_{\CC^{2+\beta}(\rho^{3+\beta})}(1 + \|\psi\|_{L^\infty(\rho)})
\]
\begin{equation}\label{eq:psi3}
\lesssim 1+\|\psi\|^{2+\varepsilon}_{L^\infty(\rho)}
\end{equation}
for some $\varepsilon\in (0,1)$.
Consequently, Lemma \ref{lemma:apriori-elliptic} implies 
\[
\|\psi\|_{L^\infty(\rho)}\lesssim 1+ \|\psi\|^{1-\varepsilon}_{L^\infty(\rho)}
\]
for some $\varepsilon\in (0,1)$. Therefore we deduce
\[
   \|\psi\|_{L^\infty(\rho)}\lesssim 1.
   \]

\subsection{Existence}
\label{ssec:ex44}

As the first step towards the  existence of solutions to the elliptic $\Phi^{4}$ model \eqref{eq:phi4} in dimension~$4$, we consider the problem on a large torus $\mathbb{T}_{M}^{4}$ of a fixed size $M\in\N$. As observed in Section \ref{ssec:decomp}, it reduces to solving the system \eqref{eq:44b}, \eqref{eq:44c} with the space white noise $\xi$ as well as the probabilistic objects $X,\llbracket X^{2}\rrbracket$ and $\llbracket X^{3}\rrbracket$ replaced by their periodic approximations $\xi_{M},X_{M},\llbracket X_{M}^{2}\rrbracket$ and $\llbracket X_{M}^{3}\rrbracket$. We refer to Section~\ref{ssec:renorm-el} for details of the probabilistic construction.

The proof of existence will be divided into two steps. First, we construct a suitable fixed point map
$$\mathcal{K}:\CC^{\beta}(\mathbb{T}_{M}^{4})\times \CC^{\beta}(\mathbb{T}_{M}^{4})\to\CC^{\beta}(\mathbb{T}_{M}^{4})\times \CC^{\beta}(\mathbb{T}_{M}^{4}).$$
Second, we apply Schaefer's fixed point theorem \cite[Section 9.2.2, Theorem 4]{E98} to  show that $\mathcal{K}$ has a fixed point.
More precisely, we define the mapping $\mathcal{K}$ as follows: given
$$(\tilde\phi,\tilde\psi)\in \CC^{\beta}(\mathbb{T}_{M}^{4})\times \CC^{\beta}(\mathbb{T}_{M}^{4}),$$
let $\mathcal{K}(\tilde\phi,\tilde\psi)=(\phi,\psi)$ be a solution to 
\begin{equation}\label{eq:44d}
\Q \phi + \Phi(\tilde\phi,\tilde\psi) = 0, \qquad \Q \psi + \psi^3 + \Psi(\tilde\phi,\tilde\psi) = 0,
\end{equation}
where
\begin{align*}
\begin{aligned}
\Phi(\tilde\phi,\tilde\psi)&= \llbracket X^3 \rrbracket + 3 (\tilde\phi + \tilde\psi) \prec \UU_{>}
   \llbracket X^2 \rrbracket + 3 (\tilde\phi + \tilde\psi)^2 \prec \UU_{>} X,\\
\Psi(\tilde\phi,\tilde\psi)&= \Psi_1(\tilde\phi,\tilde\psi) + \Psi_2(\tilde\phi,\tilde\psi), \\
 \Psi_{1}(\tilde\phi,\tilde\psi)&=\tilde\phi^3 + 3 \tilde\psi \tilde\phi^2 + 3 \tilde\psi^2 \tilde\phi,\\
\Psi_{2}(\tilde\phi,\tilde\psi)&= 3 (\tilde\phi + \tilde\psi) \prec \UU_{\leqslant} \llbracket X^2
   \rrbracket + 3 (\tilde\phi + \tilde\psi)^2 \prec \UU_{\leqslant} X + 3 (\tilde\phi +
   \tilde\psi) \succcurlyeq \llbracket X^2 \rrbracket + 3 (\tilde\phi + \tilde\psi)^2
   \succcurlyeq X.
\end{aligned}
\end{align*}   
Note that the first equation in \eqref{eq:44d} always has a (unique) solution $\phi$ which belongs to $\CC^{\alpha}(\mathbb{T}_{M}^{4})$  due to \eqref{eq:44e}. Indeed, in view of the given regularity of $(\tilde\phi,\tilde\psi)$ and the estimates from Section \ref{ssec:apr} imply  (recall that $\alpha=\beta+\kappa$)
$$
\|\Phi(\tilde\phi,\tilde\psi)\|_{\CC^{\alpha-2}(\mathbb{T}_{M}^{4})}\leq c(\|\tilde\phi+\tilde\psi\|_{\CC^{\beta}(\mathbb{T}_{M}^{4})}).
$$
Next,  we observe that due to \eqref{eq:psi1}, \eqref{eq:psi2} (performed on $\mathbb{T}^{4}_{M}$) the term $\Psi(\tilde\phi,\tilde\psi)$ belongs to $\CC^{\gamma}(\mathbb{T}^{4}_{M})$ provided   $(\tilde\phi,\tilde\psi)\in \CC^{\beta}(\mathbb{T}^{4}_{M})\times \CC^{\beta} (\mathbb{T}^{4}_{M})$ and $\gamma=\beta-\kappa$. Hence according to Proposition~\ref{prop:aux-1} there exists $\psi$ which is a unique classical solution of the second equation in \eqref{eq:44d} and belongs to $\CC^{2+\gamma}(\mathbb{T}_{M}^{4})$. This shows that the map $\mathcal{K}$ is well-defined.
As the next step, we will show that the map $\mathcal{K}$ has a fixed point.

\begin{proposition}\label{prop:fixp}
There exists $(\phi,\psi)\in\CC^{\beta}(\mathbb{T}_{M}^{4})\times \CC^{\beta}(\mathbb{T}_{M}^{4})$ such that $(\phi,\psi)=\mathcal{K}(\phi,\psi).$ Moreover, $(\phi,\psi)$ belongs to $\CC^{\alpha} (\mathbb{T}_{M}^{4})\times \CC^{2+\beta}(\mathbb{T}_{M}^{4})$ for $\alpha=\beta+\kappa$.
\end{proposition}

\begin{proof}
We intend to apply the Schaefer's fixed point theorem which can be found in \cite[Section 9.2.2, Theorem 4]{E98}. To this end, it is necessary to verify that the map $\mathcal{K}$ is continuous and compact and the set
\begin{equation}\label{eq:44f}
\{(\phi,\psi)\in\CC^{\beta}(\mathbb{T}_{M}^{4})\times \CC^{\beta}(\mathbb{T}_{M}^{4});\, (\phi,\psi)=\lambda\mathcal{K}(\phi,\psi) \text{ for some } 0\leqslant \lambda\leqslant 1\}
\end{equation}
is bounded.

\emph{Continuity and compactness:} Assume that $(\tilde\phi_{n},\tilde\psi_{n})\to(\tilde\phi,\tilde\psi)$ in $\CC^{\beta}(\mathbb{T}_{M}^{4})\times \CC^{\beta}(\mathbb{T}_{M}^{4})$ and denote $(\phi_{n},\psi_{n})=\mathcal{K}(\tilde\phi_{n},\tilde\psi_{n}).$ First, we observe that a slight modification of \eqref{eq:psi1}, \eqref{eq:psi2} and \eqref{eq:psi22} shows that
\begin{equation}\label{eq:44ha}
\|\Psi(\tilde\phi_{n},\tilde\psi_{n})\|_{\CC^{\gamma}(\mathbb{T}_{M}^{4})}\leqslant c\big(\|\tilde\phi_{n}+\tilde\psi_{n}\|_{\CC^{\beta}(\mathbb{T}_{M}^{4})}\big)\lesssim 1
\end{equation}
uniformly in $n$. Hence due to the Schauder estimates and Lemma \ref{lemma:schauder-ellptic}, it follows
\begin{equation}\label{eq:44hb}
\|\phi_{n}\|_{\CC^{\alpha}(\mathbb{T}_{M}^{4})}+\|\psi_{n}\|_{\CC^{2+\gamma}(\mathbb{T}_{M}^{4})}\lesssim 1
\end{equation}
uniformly in $n$.
According to the compact embedding \eqref{eq:emb} we deduce that there exists a subsequence still denoted by $(\phi_{n},\psi_{n})$ which converges to certain $(\phi,\psi)$ in $\CC^{\beta}(\mathbb{T}_{M}^{4})\times \CC^{\beta}(\mathbb{T}_{M}^{4})$. Moreover, due to the  uniform bound \eqref{eq:44hb}, it holds
\begin{equation*}
\|\phi\|_{\CC^{\alpha}(\mathbb{T}_{M}^{4})}+\|\psi\|_{\CC^{2+\gamma}(\mathbb{T}_{M}^{4})}\lesssim 1.
\end{equation*}
Since $\Phi$ as well as $\Psi$ in \eqref{eq:44d} depends continuously on $(\tilde\phi_{n},\tilde\psi_{n})$, which can be seen by similar estimates as in Section \ref{ssec:apr}, we may pass to the limit and conclude that $(\phi,\psi)=\mathcal{K}(\tilde\phi,\tilde\psi)$. In view of uniqueness, we deduce that every subsequence converges to the same limit which implies that the whole sequence converges and the desired  continuity of $\mathcal{K}$ follows.
Furthermore, compactness of  $\mathcal{K}$ is also a direct consequence of the bound \eqref{eq:44hb}.

\emph{Boundedness of \eqref{eq:44f}:}  If $(\phi,\psi)=\lambda\mathcal{K}(\phi,\psi) $ for some $0< \lambda\leqslant 1$, then $(\lambda^{-1}\phi,\lambda^{-1}\psi)=\mathcal{K}(\phi,\psi) $ hence
\begin{equation}\label{eq:44g}
\Q \phi + \lambda \Phi(\phi,\psi) = 0, \qquad \Q \psi + \frac{1}{\lambda^{2}}\psi^3 + \lambda\Psi(\phi,\psi) = 0.
\end{equation}
We shall modify the a priori estimates from Section \ref{ssec:apr} in order to account for the parameter $\lambda$ and obtain bounds uniform in $\lambda$. First, we observe that the first equation in \eqref{eq:44g} does not cause any difficulties as $\|\lambda\Phi(\phi,\psi)\|_{\CC^{\alpha-2}(\mathbb{T}_{M}^{4})}\leqslant \|\Phi(\phi,\psi)\|_{\CC^{\alpha-2}(\mathbb{T}_{M}^{4})}$. Consequently, as in  \eqref{eq:44e} we deduce that
$$
\|\phi\|_{\CC^{\alpha}(\mathbb{T}_{M}^{4})}\lesssim 1
$$
uniformly in $\lambda$. The same approach can be applied to the bounds \eqref{eq:psi1}, \eqref{eq:psi2} and \eqref{eq:psi22} which remain unchanged and independent of  $ \lambda$. Revisiting the proof of Lemma \ref{lemma:schauder-ellptic} we obtain
\begin{equation}\label{eq:44h}
\|\psi\|_{\CC^{2+\beta}(\mathbb{T}_{M}^{4})}\lesssim \|\Psi\|_{\CC^{\beta}(\mathbb{T}_{M}^{4})}+\frac{1}{\lambda^{2+\beta}}\|\psi\|_{L^{\infty}(\mathbb{T}_{M}^{4})}^{3+\beta}.
\end{equation}
In order to control the right hand side uniformly in $\lambda$ we revisit the proof of Lemma \ref{lemma:apriori-elliptic} and observe that it simplifies since the weight is not needed on the  torus. Then we apply  \eqref{eq:psi3} and we obtain
$$
\|\psi\|_{L^{\infty}(\mathbb{T}_{M}^{4})}\leqslant \lambda\|\Psi\|^{{1/3}}_{L^{\infty}(\mathbb{T}_{M}^{4})}\lesssim \lambda\big(1+\|\psi\|_{L^{\infty}(\mathbb{T}_{M}^{4})}^{1-\varepsilon}\big)
$$
for some $\varepsilon\in (0,1)$.
Hence, by the weighted Young inequality, we deduce
$$
\|\psi\|_{L^{\infty}(\mathbb{T}_{M}^{4})}\lesssim\lambda.
$$
Plugging this into \eqref{eq:44h} and using the bound for $\Psi$ in \eqref{eq:psi1}, \eqref{eq:psi2} leads to
$$
\|\psi\|_{\CC^{2+\beta}(\mathbb{T}_{M}^{4})}\lesssim 1,
$$
uniformly in $\lambda$ and the boundedness of \eqref{eq:44f} follows.

Finally, Schaefer's fixed point theorem \cite[Section 9.2.2, Theorem 4]{E98} gives the existence of a fixed point of $\mathcal{K}$. Moreover, the a priori estimates from Section \ref{ssec:apr} show that $\psi\in \CC^{2+\beta}(\mathbb{T}^{4}_{M}).$
\end{proof}

Therefore, we have proved the following result.

\begin{theorem}\label{thm:ex-44-per}
Let $M\in\mathbb{N}$. Let $\kappa, \alpha\in (0,1)$ be chosen sufficiently small and let $\beta=\alpha-\kappa>0$.  There exists $(\phi,\psi)\in \CC^{\alpha} (\mathbb{T}_{M}^{4})\times \CC^{2+\beta}(\mathbb{T}_{M}^{4})$ which is a solution to \eqref{eq:44b}, \eqref{eq:44c} on $\mathbb{T}_{M}^{4}$.
\end{theorem}

With this in hand, we are able to conclude the proof of existence on $\R^{4}$.

\begin{theorem}\label{thm:ex43aaa}
Let $\kappa, \alpha\in (0,1)$ be chosen sufficiently small and let $\beta=\alpha-\kappa>0$. There exists $ (\phi,\psi)\in \CC^{\alpha}(\rho)\times[\CC^{2+\beta}(\rho^{3+\beta})\cap L^{\infty}(\rho)]$  which is a solution to \eqref{eq:44b}, \eqref{eq:44c} on $\R^{4}$.
\end{theorem}

\begin{proof}
Let  $(\phi_{M},\psi_{M})\in \CC^{\alpha}(\mathbb{T}^{4}_{M})\times\CC^{2+\beta}(\mathbb{T}^{4}_{M})$ denote the solution to \eqref{eq:44b}, \eqref{eq:44c}  constructed in Theorem \ref{thm:ex-44-per}.
Since functions on $\mathbb{T}^{4}_{M}$ can be regarded as periodic functions defined on the full space $\R^{4}$, we may apply the a priori estimates from Section \ref{ssec:apr}. More precisely, in view of Theorem~\ref{thm:renorm}, we conclude that the approximate solutions $(\phi_{M},\psi_{M})$ are bounded uniformly in $M$ in $ \CC^{\alpha}(\rho)\times[\CC^{2+\beta}(\rho^{3+\beta})\cap L^{\infty}(\rho)]$ whenever $\rho$ is a polynomial bound. Due to \eqref{eq:emb}, this space is compactly embedded into $\CC^{\alpha'}(\rho^{1+\alpha'})\times\CC^{2+\beta'}(\rho^{3+\beta''})$ provided $\alpha'<\alpha$ and $\beta'<\beta<\beta''$. Therefore, there exists a subsequence, still denoted $(\phi_{M},\psi_{M})$ which converges in $\CC^{\alpha'}(\rho^{1+\alpha'})\times\CC^{2+\beta'}(\rho^{3+\beta''})$ to certain $(\phi,\psi)\in  \CC^{\alpha}(\rho)\times[\CC^{2+\beta}(\rho^{3+\beta})\cap L^{\infty}(\rho)]$. Passing to the limit in  \eqref{eq:44b}, \eqref{eq:44c} concludes the proof of existence on the full space.
\end{proof}

Finally, we note that a priori we do not know whether the solution constructed in Theorem~\ref{thm:ex43aaa} is a well-defined random variable, that is, if it is measurable with respect to $\omega$ in the underlying probability space. Indeed, the Schaefer's fixed theorem used in Proposition~\ref{prop:fixp} does not guarantee measurability. However, the existence of a measurable selection can be shown by means of Filippov's implicit function theorem \cite[Theorem 18.17]{AB06}.

\section{Elliptic $\Phi^4_5$ model}
\label{sec:45}

In this section we focus on the elliptic $\Phi^{4}$ model \eqref{eq:phi4} in dimension 5.  First we decompose the equation into a system of equations and establish a priori estimates for the involved quantities. Due to the lower regularity of the driving noise, the analysis is more involved than in Section~\ref{sec:42}. In particular, it is necessary to include additional paracontrolled ansatz, which allows to cancel certain irregular term. Consequently, the a priori estimates become rather delicate and are presented in Sections \ref{ssec:phi1}, \ref{ssec:phi2}, \ref{ssec:theta}, \ref{ssec:psi1}, \ref{ssec:psi2} below.
 This will also serve as a basis for the investigation of the parabolic $\Phi^{4}$ model in dimension 3 in Sections \ref{sec:43}, \ref{s:uniq}, \ref{s:d}.

\subsection{Decomposition into simpler equations}
\label{ssec:dec45}

We  study the elliptic equation
\begin{equation}\label{eq:45}
(- \Delta + \mu) \varphi + \varphi^3 + (- 3 a + 3b) \varphi - \xi = 0 
\end{equation}
in $\mathbbm{R}^5$ where $\xi$ is a space white noise and $a,b$ stand for renormalization constants. We let $(- \Delta +
\mu) = \Q$ and introduce the ansatz
\[ \varphi = X - X^{\!\resizebox{0.6em}{!}{
\begin{tikzpicture}
\pgfpathmoveto{\pgfqpoint{0cm}{-0.035cm}}
\pgfpathlineto{\pgfqpoint{1.376cm}{-0.035cm}}
\pgfpathlineto{\pgfqpoint{1.376cm}{1.552cm}}
\pgfpathlineto{\pgfqpoint{0cm}{1.552cm}}
\pgfpathclose
\pgfusepath{clip}
\begin{pgfscope}
\begin{pgfscope}
\pgfpathmoveto{\pgfqpoint{0cm}{-0.035cm}}
\pgfpathlineto{\pgfqpoint{1.376cm}{-0.035cm}}
\pgfpathlineto{\pgfqpoint{1.376cm}{1.552cm}}
\pgfpathlineto{\pgfqpoint{0cm}{1.552cm}}
\pgfpathclose
\pgfusepath{clip}
\begin{pgfscope}
\begin{pgfscope}
\pgfsetdash{}{0cm}
\pgfsetlinewidth{0.818mm}
\pgfsetroundcap
\pgfsetroundjoin
\pgfsetmiterlimit{7.0}
\definecolor{eps2pgf_color}{gray}{0}\pgfsetstrokecolor{eps2pgf_color}\pgfsetfillcolor{eps2pgf_color}
\pgfpathmoveto{\pgfqpoint{0.117cm}{1.421cm}}
\pgfpathlineto{\pgfqpoint{0.682cm}{0.671cm}}
\pgfpathlineto{\pgfqpoint{1.246cm}{1.421cm}}
\pgfusepath{stroke}
\end{pgfscope}
\definecolor{eps2pgf_color}{gray}{0}\pgfsetstrokecolor{eps2pgf_color}\pgfsetfillcolor{eps2pgf_color}
\pgfpathmoveto{\pgfqpoint{0.273cm}{1.395cm}}
\pgfpathcurveto{\pgfqpoint{0.273cm}{1.432cm}}{\pgfqpoint{0.259cm}{1.467cm}}{\pgfqpoint{0.233cm}{1.492cm}}
\pgfpathcurveto{\pgfqpoint{0.207cm}{1.518cm}}{\pgfqpoint{0.173cm}{1.532cm}}{\pgfqpoint{0.137cm}{1.532cm}}
\pgfpathcurveto{\pgfqpoint{0.1cm}{1.532cm}}{\pgfqpoint{0.066cm}{1.518cm}}{\pgfqpoint{0.04cm}{1.492cm}}
\pgfpathcurveto{\pgfqpoint{0.014cm}{1.467cm}}{\pgfqpoint{0cm}{1.432cm}}{\pgfqpoint{0cm}{1.395cm}}
\pgfpathcurveto{\pgfqpoint{0cm}{1.359cm}}{\pgfqpoint{0.014cm}{1.324cm}}{\pgfqpoint{0.04cm}{1.299cm}}
\pgfpathcurveto{\pgfqpoint{0.066cm}{1.273cm}}{\pgfqpoint{0.1cm}{1.258cm}}{\pgfqpoint{0.137cm}{1.258cm}}
\pgfpathcurveto{\pgfqpoint{0.173cm}{1.258cm}}{\pgfqpoint{0.207cm}{1.273cm}}{\pgfqpoint{0.233cm}{1.299cm}}
\pgfpathcurveto{\pgfqpoint{0.259cm}{1.324cm}}{\pgfqpoint{0.273cm}{1.359cm}}{\pgfqpoint{0.273cm}{1.395cm}}
\pgfusepath{fill}
\begin{pgfscope}
\pgfsetdash{}{0cm}
\pgfsetlinewidth{0.818mm}
\pgfsetmiterlimit{7.0}
\pgfpathmoveto{\pgfqpoint{0.682cm}{0.671cm}}
\pgfpathlineto{\pgfqpoint{0.679cm}{1.418cm}}
\pgfusepath{stroke}
\end{pgfscope}
\pgfpathmoveto{\pgfqpoint{0.815cm}{1.399cm}}
\pgfpathcurveto{\pgfqpoint{0.815cm}{1.435cm}}{\pgfqpoint{0.801cm}{1.47cm}}{\pgfqpoint{0.775cm}{1.496cm}}
\pgfpathcurveto{\pgfqpoint{0.75cm}{1.521cm}}{\pgfqpoint{0.715cm}{1.536cm}}{\pgfqpoint{0.679cm}{1.536cm}}
\pgfpathcurveto{\pgfqpoint{0.643cm}{1.536cm}}{\pgfqpoint{0.608cm}{1.521cm}}{\pgfqpoint{0.582cm}{1.496cm}}
\pgfpathcurveto{\pgfqpoint{0.557cm}{1.47cm}}{\pgfqpoint{0.542cm}{1.435cm}}{\pgfqpoint{0.542cm}{1.399cm}}
\pgfpathcurveto{\pgfqpoint{0.542cm}{1.363cm}}{\pgfqpoint{0.557cm}{1.328cm}}{\pgfqpoint{0.582cm}{1.302cm}}
\pgfpathcurveto{\pgfqpoint{0.608cm}{1.276cm}}{\pgfqpoint{0.643cm}{1.262cm}}{\pgfqpoint{0.679cm}{1.262cm}}
\pgfpathcurveto{\pgfqpoint{0.715cm}{1.262cm}}{\pgfqpoint{0.75cm}{1.276cm}}{\pgfqpoint{0.775cm}{1.302cm}}
\pgfpathcurveto{\pgfqpoint{0.801cm}{1.328cm}}{\pgfqpoint{0.815cm}{1.363cm}}{\pgfqpoint{0.815cm}{1.399cm}}
\pgfusepath{fill}
\pgfpathmoveto{\pgfqpoint{1.345cm}{1.371cm}}
\pgfpathcurveto{\pgfqpoint{1.345cm}{1.408cm}}{\pgfqpoint{1.331cm}{1.442cm}}{\pgfqpoint{1.305cm}{1.468cm}}
\pgfpathcurveto{\pgfqpoint{1.28cm}{1.494cm}}{\pgfqpoint{1.245cm}{1.508cm}}{\pgfqpoint{1.209cm}{1.508cm}}
\pgfpathcurveto{\pgfqpoint{1.172cm}{1.508cm}}{\pgfqpoint{1.138cm}{1.494cm}}{\pgfqpoint{1.112cm}{1.468cm}}
\pgfpathcurveto{\pgfqpoint{1.087cm}{1.442cm}}{\pgfqpoint{1.072cm}{1.408cm}}{\pgfqpoint{1.072cm}{1.371cm}}
\pgfpathcurveto{\pgfqpoint{1.072cm}{1.335cm}}{\pgfqpoint{1.087cm}{1.3cm}}{\pgfqpoint{1.112cm}{1.274cm}}
\pgfpathcurveto{\pgfqpoint{1.138cm}{1.249cm}}{\pgfqpoint{1.172cm}{1.234cm}}{\pgfqpoint{1.209cm}{1.234cm}}
\pgfpathcurveto{\pgfqpoint{1.245cm}{1.234cm}}{\pgfqpoint{1.28cm}{1.249cm}}{\pgfqpoint{1.305cm}{1.274cm}}
\pgfpathcurveto{\pgfqpoint{1.331cm}{1.3cm}}{\pgfqpoint{1.345cm}{1.335cm}}{\pgfqpoint{1.345cm}{1.371cm}}
\pgfusepath{fill}
\begin{pgfscope}
\pgfsetdash{}{0cm}
\pgfsetlinewidth{0.818mm}
\pgfsetroundcap
\pgfsetmiterlimit{4.0}
\pgfpathmoveto{\pgfqpoint{0.682cm}{0.671cm}}
\pgfpathlineto{\pgfqpoint{0.682cm}{0.042cm}}
\pgfusepath{stroke}
\end{pgfscope}
\end{pgfscope}
\end{pgfscope}
\end{pgfscope}
\end{tikzpicture}}} + \phi + \psi\]
with
\[ \Q X = \xi, \quad \llbracket X^3 \rrbracket \assign X^3 - 3 a X, \quad \llbracket X^2
   \rrbracket \assign X^2 - a, \]
\[ \Q X^{\!\resizebox{0.6em}{!}{
\begin{tikzpicture}
\pgfpathmoveto{\pgfqpoint{0cm}{-0.035cm}}
\pgfpathlineto{\pgfqpoint{1.376cm}{-0.035cm}}
\pgfpathlineto{\pgfqpoint{1.376cm}{1.552cm}}
\pgfpathlineto{\pgfqpoint{0cm}{1.552cm}}
\pgfpathclose
\pgfusepath{clip}
\begin{pgfscope}
\begin{pgfscope}
\pgfpathmoveto{\pgfqpoint{0cm}{-0.035cm}}
\pgfpathlineto{\pgfqpoint{1.376cm}{-0.035cm}}
\pgfpathlineto{\pgfqpoint{1.376cm}{1.552cm}}
\pgfpathlineto{\pgfqpoint{0cm}{1.552cm}}
\pgfpathclose
\pgfusepath{clip}
\begin{pgfscope}
\begin{pgfscope}
\pgfsetdash{}{0cm}
\pgfsetlinewidth{0.818mm}
\pgfsetroundcap
\pgfsetroundjoin
\pgfsetmiterlimit{7.0}
\definecolor{eps2pgf_color}{gray}{0}\pgfsetstrokecolor{eps2pgf_color}\pgfsetfillcolor{eps2pgf_color}
\pgfpathmoveto{\pgfqpoint{0.117cm}{1.421cm}}
\pgfpathlineto{\pgfqpoint{0.682cm}{0.671cm}}
\pgfpathlineto{\pgfqpoint{1.246cm}{1.421cm}}
\pgfusepath{stroke}
\end{pgfscope}
\definecolor{eps2pgf_color}{gray}{0}\pgfsetstrokecolor{eps2pgf_color}\pgfsetfillcolor{eps2pgf_color}
\pgfpathmoveto{\pgfqpoint{0.273cm}{1.395cm}}
\pgfpathcurveto{\pgfqpoint{0.273cm}{1.432cm}}{\pgfqpoint{0.259cm}{1.467cm}}{\pgfqpoint{0.233cm}{1.492cm}}
\pgfpathcurveto{\pgfqpoint{0.207cm}{1.518cm}}{\pgfqpoint{0.173cm}{1.532cm}}{\pgfqpoint{0.137cm}{1.532cm}}
\pgfpathcurveto{\pgfqpoint{0.1cm}{1.532cm}}{\pgfqpoint{0.066cm}{1.518cm}}{\pgfqpoint{0.04cm}{1.492cm}}
\pgfpathcurveto{\pgfqpoint{0.014cm}{1.467cm}}{\pgfqpoint{0cm}{1.432cm}}{\pgfqpoint{0cm}{1.395cm}}
\pgfpathcurveto{\pgfqpoint{0cm}{1.359cm}}{\pgfqpoint{0.014cm}{1.324cm}}{\pgfqpoint{0.04cm}{1.299cm}}
\pgfpathcurveto{\pgfqpoint{0.066cm}{1.273cm}}{\pgfqpoint{0.1cm}{1.258cm}}{\pgfqpoint{0.137cm}{1.258cm}}
\pgfpathcurveto{\pgfqpoint{0.173cm}{1.258cm}}{\pgfqpoint{0.207cm}{1.273cm}}{\pgfqpoint{0.233cm}{1.299cm}}
\pgfpathcurveto{\pgfqpoint{0.259cm}{1.324cm}}{\pgfqpoint{0.273cm}{1.359cm}}{\pgfqpoint{0.273cm}{1.395cm}}
\pgfusepath{fill}
\begin{pgfscope}
\pgfsetdash{}{0cm}
\pgfsetlinewidth{0.818mm}
\pgfsetmiterlimit{7.0}
\pgfpathmoveto{\pgfqpoint{0.682cm}{0.671cm}}
\pgfpathlineto{\pgfqpoint{0.679cm}{1.418cm}}
\pgfusepath{stroke}
\end{pgfscope}
\pgfpathmoveto{\pgfqpoint{0.815cm}{1.399cm}}
\pgfpathcurveto{\pgfqpoint{0.815cm}{1.435cm}}{\pgfqpoint{0.801cm}{1.47cm}}{\pgfqpoint{0.775cm}{1.496cm}}
\pgfpathcurveto{\pgfqpoint{0.75cm}{1.521cm}}{\pgfqpoint{0.715cm}{1.536cm}}{\pgfqpoint{0.679cm}{1.536cm}}
\pgfpathcurveto{\pgfqpoint{0.643cm}{1.536cm}}{\pgfqpoint{0.608cm}{1.521cm}}{\pgfqpoint{0.582cm}{1.496cm}}
\pgfpathcurveto{\pgfqpoint{0.557cm}{1.47cm}}{\pgfqpoint{0.542cm}{1.435cm}}{\pgfqpoint{0.542cm}{1.399cm}}
\pgfpathcurveto{\pgfqpoint{0.542cm}{1.363cm}}{\pgfqpoint{0.557cm}{1.328cm}}{\pgfqpoint{0.582cm}{1.302cm}}
\pgfpathcurveto{\pgfqpoint{0.608cm}{1.276cm}}{\pgfqpoint{0.643cm}{1.262cm}}{\pgfqpoint{0.679cm}{1.262cm}}
\pgfpathcurveto{\pgfqpoint{0.715cm}{1.262cm}}{\pgfqpoint{0.75cm}{1.276cm}}{\pgfqpoint{0.775cm}{1.302cm}}
\pgfpathcurveto{\pgfqpoint{0.801cm}{1.328cm}}{\pgfqpoint{0.815cm}{1.363cm}}{\pgfqpoint{0.815cm}{1.399cm}}
\pgfusepath{fill}
\pgfpathmoveto{\pgfqpoint{1.345cm}{1.371cm}}
\pgfpathcurveto{\pgfqpoint{1.345cm}{1.408cm}}{\pgfqpoint{1.331cm}{1.442cm}}{\pgfqpoint{1.305cm}{1.468cm}}
\pgfpathcurveto{\pgfqpoint{1.28cm}{1.494cm}}{\pgfqpoint{1.245cm}{1.508cm}}{\pgfqpoint{1.209cm}{1.508cm}}
\pgfpathcurveto{\pgfqpoint{1.172cm}{1.508cm}}{\pgfqpoint{1.138cm}{1.494cm}}{\pgfqpoint{1.112cm}{1.468cm}}
\pgfpathcurveto{\pgfqpoint{1.087cm}{1.442cm}}{\pgfqpoint{1.072cm}{1.408cm}}{\pgfqpoint{1.072cm}{1.371cm}}
\pgfpathcurveto{\pgfqpoint{1.072cm}{1.335cm}}{\pgfqpoint{1.087cm}{1.3cm}}{\pgfqpoint{1.112cm}{1.274cm}}
\pgfpathcurveto{\pgfqpoint{1.138cm}{1.249cm}}{\pgfqpoint{1.172cm}{1.234cm}}{\pgfqpoint{1.209cm}{1.234cm}}
\pgfpathcurveto{\pgfqpoint{1.245cm}{1.234cm}}{\pgfqpoint{1.28cm}{1.249cm}}{\pgfqpoint{1.305cm}{1.274cm}}
\pgfpathcurveto{\pgfqpoint{1.331cm}{1.3cm}}{\pgfqpoint{1.345cm}{1.335cm}}{\pgfqpoint{1.345cm}{1.371cm}}
\pgfusepath{fill}
\begin{pgfscope}
\pgfsetdash{}{0cm}
\pgfsetlinewidth{0.818mm}
\pgfsetroundcap
\pgfsetmiterlimit{4.0}
\pgfpathmoveto{\pgfqpoint{0.682cm}{0.671cm}}
\pgfpathlineto{\pgfqpoint{0.682cm}{0.042cm}}
\pgfusepath{stroke}
\end{pgfscope}
\end{pgfscope}
\end{pgfscope}
\end{pgfscope}
\end{tikzpicture}}} = \llbracket X^3 \rrbracket, \quad \Q
   X^{\!\resizebox{0.6em}{!}{
\begin{tikzpicture}
\pgfpathmoveto{\pgfqpoint{0cm}{0cm}}
\pgfpathlineto{\pgfqpoint{1.376cm}{0cm}}
\pgfpathlineto{\pgfqpoint{1.376cm}{1.588cm}}
\pgfpathlineto{\pgfqpoint{0cm}{1.588cm}}
\pgfpathclose
\pgfusepath{clip}
\begin{pgfscope}
\begin{pgfscope}
\pgfpathmoveto{\pgfqpoint{0cm}{0cm}}
\pgfpathlineto{\pgfqpoint{1.376cm}{0cm}}
\pgfpathlineto{\pgfqpoint{1.376cm}{1.588cm}}
\pgfpathlineto{\pgfqpoint{0cm}{1.588cm}}
\pgfpathclose
\pgfusepath{clip}
\begin{pgfscope}
\begin{pgfscope}
\definecolor{eps2pgf_color}{gray}{0.976471}\pgfsetstrokecolor{eps2pgf_color}\pgfsetfillcolor{eps2pgf_color}
\pgfpathmoveto{\pgfqpoint{0cm}{0cm}}
\pgfpathlineto{\pgfqpoint{1.376cm}{0cm}}
\pgfpathlineto{\pgfqpoint{1.376cm}{1.588cm}}
\pgfpathlineto{\pgfqpoint{0cm}{1.588cm}}
\pgfpathclose
\pgfusepath{fill}
\end{pgfscope}
\begin{pgfscope}
\pgfsetdash{}{0cm}
\pgfsetlinewidth{0.818mm}
\pgfsetroundcap
\pgfsetroundjoin
\pgfsetmiterlimit{7.0}
\definecolor{eps2pgf_color}{gray}{0}\pgfsetstrokecolor{eps2pgf_color}\pgfsetfillcolor{eps2pgf_color}
\pgfpathmoveto{\pgfqpoint{0.117cm}{1.476cm}}
\pgfpathlineto{\pgfqpoint{0.682cm}{0.726cm}}
\pgfpathlineto{\pgfqpoint{1.246cm}{1.476cm}}
\pgfusepath{stroke}
\end{pgfscope}
\definecolor{eps2pgf_color}{gray}{0}\pgfsetstrokecolor{eps2pgf_color}\pgfsetfillcolor{eps2pgf_color}
\pgfpathmoveto{\pgfqpoint{0.273cm}{1.451cm}}
\pgfpathcurveto{\pgfqpoint{0.273cm}{1.487cm}}{\pgfqpoint{0.259cm}{1.522cm}}{\pgfqpoint{0.233cm}{1.547cm}}
\pgfpathcurveto{\pgfqpoint{0.207cm}{1.573cm}}{\pgfqpoint{0.173cm}{1.588cm}}{\pgfqpoint{0.137cm}{1.588cm}}
\pgfpathcurveto{\pgfqpoint{0.1cm}{1.588cm}}{\pgfqpoint{0.066cm}{1.573cm}}{\pgfqpoint{0.04cm}{1.547cm}}
\pgfpathcurveto{\pgfqpoint{0.014cm}{1.522cm}}{\pgfqpoint{0cm}{1.487cm}}{\pgfqpoint{0cm}{1.451cm}}
\pgfpathcurveto{\pgfqpoint{0cm}{1.414cm}}{\pgfqpoint{0.014cm}{1.379cm}}{\pgfqpoint{0.04cm}{1.354cm}}
\pgfpathcurveto{\pgfqpoint{0.066cm}{1.328cm}}{\pgfqpoint{0.1cm}{1.314cm}}{\pgfqpoint{0.137cm}{1.314cm}}
\pgfpathcurveto{\pgfqpoint{0.173cm}{1.314cm}}{\pgfqpoint{0.207cm}{1.328cm}}{\pgfqpoint{0.233cm}{1.354cm}}
\pgfpathcurveto{\pgfqpoint{0.259cm}{1.379cm}}{\pgfqpoint{0.273cm}{1.414cm}}{\pgfqpoint{0.273cm}{1.451cm}}
\pgfusepath{fill}
\pgfpathmoveto{\pgfqpoint{1.345cm}{1.426cm}}
\pgfpathcurveto{\pgfqpoint{1.345cm}{1.463cm}}{\pgfqpoint{1.331cm}{1.497cm}}{\pgfqpoint{1.305cm}{1.523cm}}
\pgfpathcurveto{\pgfqpoint{1.28cm}{1.549cm}}{\pgfqpoint{1.245cm}{1.563cm}}{\pgfqpoint{1.209cm}{1.563cm}}
\pgfpathcurveto{\pgfqpoint{1.172cm}{1.563cm}}{\pgfqpoint{1.138cm}{1.549cm}}{\pgfqpoint{1.112cm}{1.523cm}}
\pgfpathcurveto{\pgfqpoint{1.087cm}{1.497cm}}{\pgfqpoint{1.072cm}{1.463cm}}{\pgfqpoint{1.072cm}{1.426cm}}
\pgfpathcurveto{\pgfqpoint{1.072cm}{1.39cm}}{\pgfqpoint{1.087cm}{1.355cm}}{\pgfqpoint{1.112cm}{1.329cm}}
\pgfpathcurveto{\pgfqpoint{1.138cm}{1.304cm}}{\pgfqpoint{1.172cm}{1.289cm}}{\pgfqpoint{1.209cm}{1.289cm}}
\pgfpathcurveto{\pgfqpoint{1.245cm}{1.289cm}}{\pgfqpoint{1.28cm}{1.304cm}}{\pgfqpoint{1.305cm}{1.329cm}}
\pgfpathcurveto{\pgfqpoint{1.331cm}{1.355cm}}{\pgfqpoint{1.345cm}{1.39cm}}{\pgfqpoint{1.345cm}{1.426cm}}
\pgfusepath{fill}
\begin{pgfscope}
\pgfsetdash{}{0cm}
\pgfsetlinewidth{0.818mm}
\pgfsetroundcap
\pgfsetmiterlimit{4.0}
\pgfpathmoveto{\pgfqpoint{0.682cm}{0.726cm}}
\pgfpathlineto{\pgfqpoint{0.682cm}{0.097cm}}
\pgfusepath{stroke}
\end{pgfscope}
\end{pgfscope}
\end{pgfscope}
\end{pgfscope}
\end{tikzpicture}}} = \llbracket X^2 \rrbracket, \]
\[X^{\!\resizebox{!}{.8em}{
\begin{tikzpicture}
\pgfpathmoveto{\pgfqpoint{0cm}{-0.035cm}}
\pgfpathlineto{\pgfqpoint{1.976cm}{-0.035cm}}
\pgfpathlineto{\pgfqpoint{1.976cm}{1.94cm}}
\pgfpathlineto{\pgfqpoint{0cm}{1.94cm}}
\pgfpathclose
\pgfusepath{clip}
\begin{pgfscope}
\begin{pgfscope}
\pgfpathmoveto{\pgfqpoint{0cm}{-0.035cm}}
\pgfpathlineto{\pgfqpoint{1.976cm}{-0.035cm}}
\pgfpathlineto{\pgfqpoint{1.976cm}{1.94cm}}
\pgfpathlineto{\pgfqpoint{0cm}{1.94cm}}
\pgfpathclose
\pgfusepath{clip}
\begin{pgfscope}
\begin{pgfscope}
\pgfsetdash{}{0cm}
\pgfsetlinewidth{0.818mm}
\pgfsetroundcap
\pgfsetroundjoin
\pgfsetmiterlimit{7.0}
\definecolor{eps2pgf_color}{gray}{0}\pgfsetstrokecolor{eps2pgf_color}\pgfsetfillcolor{eps2pgf_color}
\pgfpathmoveto{\pgfqpoint{0.117cm}{1.815cm}}
\pgfpathlineto{\pgfqpoint{0.682cm}{1.065cm}}
\pgfpathlineto{\pgfqpoint{1.246cm}{1.815cm}}
\pgfusepath{stroke}
\end{pgfscope}
\definecolor{eps2pgf_color}{gray}{0}\pgfsetstrokecolor{eps2pgf_color}\pgfsetfillcolor{eps2pgf_color}
\pgfpathmoveto{\pgfqpoint{0.273cm}{1.789cm}}
\pgfpathcurveto{\pgfqpoint{0.273cm}{1.825cm}}{\pgfqpoint{0.259cm}{1.86cm}}{\pgfqpoint{0.233cm}{1.886cm}}
\pgfpathcurveto{\pgfqpoint{0.207cm}{1.912cm}}{\pgfqpoint{0.173cm}{1.926cm}}{\pgfqpoint{0.137cm}{1.926cm}}
\pgfpathcurveto{\pgfqpoint{0.1cm}{1.926cm}}{\pgfqpoint{0.066cm}{1.912cm}}{\pgfqpoint{0.04cm}{1.886cm}}
\pgfpathcurveto{\pgfqpoint{0.014cm}{1.86cm}}{\pgfqpoint{0cm}{1.825cm}}{\pgfqpoint{0cm}{1.789cm}}
\pgfpathcurveto{\pgfqpoint{0cm}{1.753cm}}{\pgfqpoint{0.014cm}{1.718cm}}{\pgfqpoint{0.04cm}{1.692cm}}
\pgfpathcurveto{\pgfqpoint{0.066cm}{1.667cm}}{\pgfqpoint{0.1cm}{1.652cm}}{\pgfqpoint{0.137cm}{1.652cm}}
\pgfpathcurveto{\pgfqpoint{0.173cm}{1.652cm}}{\pgfqpoint{0.207cm}{1.667cm}}{\pgfqpoint{0.233cm}{1.692cm}}
\pgfpathcurveto{\pgfqpoint{0.259cm}{1.718cm}}{\pgfqpoint{0.273cm}{1.753cm}}{\pgfqpoint{0.273cm}{1.789cm}}
\pgfusepath{fill}
\begin{pgfscope}
\pgfsetdash{}{0cm}
\pgfsetlinewidth{0.818mm}
\pgfsetmiterlimit{7.0}
\pgfpathmoveto{\pgfqpoint{0.682cm}{1.065cm}}
\pgfpathlineto{\pgfqpoint{0.679cm}{1.812cm}}
\pgfusepath{stroke}
\end{pgfscope}
\pgfpathmoveto{\pgfqpoint{0.815cm}{1.793cm}}
\pgfpathcurveto{\pgfqpoint{0.815cm}{1.829cm}}{\pgfqpoint{0.801cm}{1.864cm}}{\pgfqpoint{0.775cm}{1.89cm}}
\pgfpathcurveto{\pgfqpoint{0.75cm}{1.915cm}}{\pgfqpoint{0.715cm}{1.93cm}}{\pgfqpoint{0.679cm}{1.93cm}}
\pgfpathcurveto{\pgfqpoint{0.643cm}{1.93cm}}{\pgfqpoint{0.608cm}{1.915cm}}{\pgfqpoint{0.582cm}{1.89cm}}
\pgfpathcurveto{\pgfqpoint{0.557cm}{1.864cm}}{\pgfqpoint{0.542cm}{1.829cm}}{\pgfqpoint{0.542cm}{1.793cm}}
\pgfpathcurveto{\pgfqpoint{0.542cm}{1.756cm}}{\pgfqpoint{0.557cm}{1.722cm}}{\pgfqpoint{0.582cm}{1.696cm}}
\pgfpathcurveto{\pgfqpoint{0.608cm}{1.67cm}}{\pgfqpoint{0.643cm}{1.656cm}}{\pgfqpoint{0.679cm}{1.656cm}}
\pgfpathcurveto{\pgfqpoint{0.715cm}{1.656cm}}{\pgfqpoint{0.75cm}{1.67cm}}{\pgfqpoint{0.775cm}{1.696cm}}
\pgfpathcurveto{\pgfqpoint{0.801cm}{1.722cm}}{\pgfqpoint{0.815cm}{1.756cm}}{\pgfqpoint{0.815cm}{1.793cm}}
\pgfusepath{fill}
\pgfpathmoveto{\pgfqpoint{1.345cm}{1.765cm}}
\pgfpathcurveto{\pgfqpoint{1.345cm}{1.801cm}}{\pgfqpoint{1.331cm}{1.836cm}}{\pgfqpoint{1.305cm}{1.862cm}}
\pgfpathcurveto{\pgfqpoint{1.28cm}{1.887cm}}{\pgfqpoint{1.245cm}{1.902cm}}{\pgfqpoint{1.209cm}{1.902cm}}
\pgfpathcurveto{\pgfqpoint{1.172cm}{1.902cm}}{\pgfqpoint{1.138cm}{1.887cm}}{\pgfqpoint{1.112cm}{1.862cm}}
\pgfpathcurveto{\pgfqpoint{1.087cm}{1.836cm}}{\pgfqpoint{1.072cm}{1.801cm}}{\pgfqpoint{1.072cm}{1.765cm}}
\pgfpathcurveto{\pgfqpoint{1.072cm}{1.728cm}}{\pgfqpoint{1.087cm}{1.694cm}}{\pgfqpoint{1.112cm}{1.668cm}}
\pgfpathcurveto{\pgfqpoint{1.138cm}{1.642cm}}{\pgfqpoint{1.172cm}{1.628cm}}{\pgfqpoint{1.209cm}{1.628cm}}
\pgfpathcurveto{\pgfqpoint{1.245cm}{1.628cm}}{\pgfqpoint{1.28cm}{1.642cm}}{\pgfqpoint{1.305cm}{1.668cm}}
\pgfpathcurveto{\pgfqpoint{1.331cm}{1.694cm}}{\pgfqpoint{1.345cm}{1.728cm}}{\pgfqpoint{1.345cm}{1.765cm}}
\pgfusepath{fill}
\begin{pgfscope}
\pgfsetdash{}{0cm}
\pgfsetlinewidth{0.818mm}
\pgfsetroundcap
\pgfsetroundjoin
\pgfsetmiterlimit{7.0}
\pgfpathmoveto{\pgfqpoint{0.682cm}{1.065cm}}
\pgfpathlineto{\pgfqpoint{1.246cm}{0.315cm}}
\pgfpathlineto{\pgfqpoint{1.811cm}{1.065cm}}
\pgfusepath{stroke}
\end{pgfscope}
\pgfpathmoveto{\pgfqpoint{1.948cm}{1.065cm}}
\pgfpathcurveto{\pgfqpoint{1.948cm}{1.101cm}}{\pgfqpoint{1.933cm}{1.136cm}}{\pgfqpoint{1.907cm}{1.162cm}}
\pgfpathcurveto{\pgfqpoint{1.882cm}{1.187cm}}{\pgfqpoint{1.847cm}{1.202cm}}{\pgfqpoint{1.811cm}{1.202cm}}
\pgfpathcurveto{\pgfqpoint{1.775cm}{1.202cm}}{\pgfqpoint{1.74cm}{1.187cm}}{\pgfqpoint{1.714cm}{1.162cm}}
\pgfpathcurveto{\pgfqpoint{1.689cm}{1.136cm}}{\pgfqpoint{1.674cm}{1.101cm}}{\pgfqpoint{1.674cm}{1.065cm}}
\pgfpathcurveto{\pgfqpoint{1.674cm}{1.029cm}}{\pgfqpoint{1.689cm}{0.994cm}}{\pgfqpoint{1.714cm}{0.968cm}}
\pgfpathcurveto{\pgfqpoint{1.74cm}{0.942cm}}{\pgfqpoint{1.775cm}{0.928cm}}{\pgfqpoint{1.811cm}{0.928cm}}
\pgfpathcurveto{\pgfqpoint{1.847cm}{0.928cm}}{\pgfqpoint{1.882cm}{0.942cm}}{\pgfqpoint{1.907cm}{0.968cm}}
\pgfpathcurveto{\pgfqpoint{1.933cm}{0.994cm}}{\pgfqpoint{1.948cm}{1.029cm}}{\pgfqpoint{1.948cm}{1.065cm}}
\pgfusepath{fill}
\begin{pgfscope}
\pgfsetdash{}{0cm}
\pgfsetlinewidth{0.818mm}
\pgfsetmiterlimit{4.0}
\pgfpathmoveto{\pgfqpoint{1.383cm}{0.178cm}}
\pgfpathcurveto{\pgfqpoint{1.383cm}{0.214cm}}{\pgfqpoint{1.369cm}{0.249cm}}{\pgfqpoint{1.343cm}{0.275cm}}
\pgfpathcurveto{\pgfqpoint{1.317cm}{0.3cm}}{\pgfqpoint{1.283cm}{0.315cm}}{\pgfqpoint{1.246cm}{0.315cm}}
\pgfpathcurveto{\pgfqpoint{1.21cm}{0.315cm}}{\pgfqpoint{1.175cm}{0.3cm}}{\pgfqpoint{1.15cm}{0.275cm}}
\pgfpathcurveto{\pgfqpoint{1.124cm}{0.249cm}}{\pgfqpoint{1.11cm}{0.214cm}}{\pgfqpoint{1.11cm}{0.178cm}}
\pgfpathcurveto{\pgfqpoint{1.11cm}{0.141cm}}{\pgfqpoint{1.124cm}{0.107cm}}{\pgfqpoint{1.15cm}{0.081cm}}
\pgfpathcurveto{\pgfqpoint{1.175cm}{0.055cm}}{\pgfqpoint{1.21cm}{0.041cm}}{\pgfqpoint{1.246cm}{0.041cm}}
\pgfpathcurveto{\pgfqpoint{1.283cm}{0.041cm}}{\pgfqpoint{1.317cm}{0.055cm}}{\pgfqpoint{1.343cm}{0.081cm}}
\pgfpathcurveto{\pgfqpoint{1.369cm}{0.107cm}}{\pgfqpoint{1.383cm}{0.141cm}}{\pgfqpoint{1.383cm}{0.178cm}}
\pgfusepath{stroke}
\end{pgfscope}
\end{pgfscope}
\end{pgfscope}
\end{pgfscope}
\end{tikzpicture}}}=X^{\!\resizebox{0.6em}{!}{
\begin{tikzpicture}
\pgfpathmoveto{\pgfqpoint{0cm}{-0.035cm}}
\pgfpathlineto{\pgfqpoint{1.376cm}{-0.035cm}}
\pgfpathlineto{\pgfqpoint{1.376cm}{1.552cm}}
\pgfpathlineto{\pgfqpoint{0cm}{1.552cm}}
\pgfpathclose
\pgfusepath{clip}
\begin{pgfscope}
\begin{pgfscope}
\pgfpathmoveto{\pgfqpoint{0cm}{-0.035cm}}
\pgfpathlineto{\pgfqpoint{1.376cm}{-0.035cm}}
\pgfpathlineto{\pgfqpoint{1.376cm}{1.552cm}}
\pgfpathlineto{\pgfqpoint{0cm}{1.552cm}}
\pgfpathclose
\pgfusepath{clip}
\begin{pgfscope}
\begin{pgfscope}
\pgfsetdash{}{0cm}
\pgfsetlinewidth{0.818mm}
\pgfsetroundcap
\pgfsetroundjoin
\pgfsetmiterlimit{7.0}
\definecolor{eps2pgf_color}{gray}{0}\pgfsetstrokecolor{eps2pgf_color}\pgfsetfillcolor{eps2pgf_color}
\pgfpathmoveto{\pgfqpoint{0.117cm}{1.421cm}}
\pgfpathlineto{\pgfqpoint{0.682cm}{0.671cm}}
\pgfpathlineto{\pgfqpoint{1.246cm}{1.421cm}}
\pgfusepath{stroke}
\end{pgfscope}
\definecolor{eps2pgf_color}{gray}{0}\pgfsetstrokecolor{eps2pgf_color}\pgfsetfillcolor{eps2pgf_color}
\pgfpathmoveto{\pgfqpoint{0.273cm}{1.395cm}}
\pgfpathcurveto{\pgfqpoint{0.273cm}{1.432cm}}{\pgfqpoint{0.259cm}{1.467cm}}{\pgfqpoint{0.233cm}{1.492cm}}
\pgfpathcurveto{\pgfqpoint{0.207cm}{1.518cm}}{\pgfqpoint{0.173cm}{1.532cm}}{\pgfqpoint{0.137cm}{1.532cm}}
\pgfpathcurveto{\pgfqpoint{0.1cm}{1.532cm}}{\pgfqpoint{0.066cm}{1.518cm}}{\pgfqpoint{0.04cm}{1.492cm}}
\pgfpathcurveto{\pgfqpoint{0.014cm}{1.467cm}}{\pgfqpoint{0cm}{1.432cm}}{\pgfqpoint{0cm}{1.395cm}}
\pgfpathcurveto{\pgfqpoint{0cm}{1.359cm}}{\pgfqpoint{0.014cm}{1.324cm}}{\pgfqpoint{0.04cm}{1.299cm}}
\pgfpathcurveto{\pgfqpoint{0.066cm}{1.273cm}}{\pgfqpoint{0.1cm}{1.258cm}}{\pgfqpoint{0.137cm}{1.258cm}}
\pgfpathcurveto{\pgfqpoint{0.173cm}{1.258cm}}{\pgfqpoint{0.207cm}{1.273cm}}{\pgfqpoint{0.233cm}{1.299cm}}
\pgfpathcurveto{\pgfqpoint{0.259cm}{1.324cm}}{\pgfqpoint{0.273cm}{1.359cm}}{\pgfqpoint{0.273cm}{1.395cm}}
\pgfusepath{fill}
\begin{pgfscope}
\pgfsetdash{}{0cm}
\pgfsetlinewidth{0.818mm}
\pgfsetmiterlimit{7.0}
\pgfpathmoveto{\pgfqpoint{0.682cm}{0.671cm}}
\pgfpathlineto{\pgfqpoint{0.679cm}{1.418cm}}
\pgfusepath{stroke}
\end{pgfscope}
\pgfpathmoveto{\pgfqpoint{0.815cm}{1.399cm}}
\pgfpathcurveto{\pgfqpoint{0.815cm}{1.435cm}}{\pgfqpoint{0.801cm}{1.47cm}}{\pgfqpoint{0.775cm}{1.496cm}}
\pgfpathcurveto{\pgfqpoint{0.75cm}{1.521cm}}{\pgfqpoint{0.715cm}{1.536cm}}{\pgfqpoint{0.679cm}{1.536cm}}
\pgfpathcurveto{\pgfqpoint{0.643cm}{1.536cm}}{\pgfqpoint{0.608cm}{1.521cm}}{\pgfqpoint{0.582cm}{1.496cm}}
\pgfpathcurveto{\pgfqpoint{0.557cm}{1.47cm}}{\pgfqpoint{0.542cm}{1.435cm}}{\pgfqpoint{0.542cm}{1.399cm}}
\pgfpathcurveto{\pgfqpoint{0.542cm}{1.363cm}}{\pgfqpoint{0.557cm}{1.328cm}}{\pgfqpoint{0.582cm}{1.302cm}}
\pgfpathcurveto{\pgfqpoint{0.608cm}{1.276cm}}{\pgfqpoint{0.643cm}{1.262cm}}{\pgfqpoint{0.679cm}{1.262cm}}
\pgfpathcurveto{\pgfqpoint{0.715cm}{1.262cm}}{\pgfqpoint{0.75cm}{1.276cm}}{\pgfqpoint{0.775cm}{1.302cm}}
\pgfpathcurveto{\pgfqpoint{0.801cm}{1.328cm}}{\pgfqpoint{0.815cm}{1.363cm}}{\pgfqpoint{0.815cm}{1.399cm}}
\pgfusepath{fill}
\pgfpathmoveto{\pgfqpoint{1.345cm}{1.371cm}}
\pgfpathcurveto{\pgfqpoint{1.345cm}{1.408cm}}{\pgfqpoint{1.331cm}{1.442cm}}{\pgfqpoint{1.305cm}{1.468cm}}
\pgfpathcurveto{\pgfqpoint{1.28cm}{1.494cm}}{\pgfqpoint{1.245cm}{1.508cm}}{\pgfqpoint{1.209cm}{1.508cm}}
\pgfpathcurveto{\pgfqpoint{1.172cm}{1.508cm}}{\pgfqpoint{1.138cm}{1.494cm}}{\pgfqpoint{1.112cm}{1.468cm}}
\pgfpathcurveto{\pgfqpoint{1.087cm}{1.442cm}}{\pgfqpoint{1.072cm}{1.408cm}}{\pgfqpoint{1.072cm}{1.371cm}}
\pgfpathcurveto{\pgfqpoint{1.072cm}{1.335cm}}{\pgfqpoint{1.087cm}{1.3cm}}{\pgfqpoint{1.112cm}{1.274cm}}
\pgfpathcurveto{\pgfqpoint{1.138cm}{1.249cm}}{\pgfqpoint{1.172cm}{1.234cm}}{\pgfqpoint{1.209cm}{1.234cm}}
\pgfpathcurveto{\pgfqpoint{1.245cm}{1.234cm}}{\pgfqpoint{1.28cm}{1.249cm}}{\pgfqpoint{1.305cm}{1.274cm}}
\pgfpathcurveto{\pgfqpoint{1.331cm}{1.3cm}}{\pgfqpoint{1.345cm}{1.335cm}}{\pgfqpoint{1.345cm}{1.371cm}}
\pgfusepath{fill}
\begin{pgfscope}
\pgfsetdash{}{0cm}
\pgfsetlinewidth{0.818mm}
\pgfsetroundcap
\pgfsetmiterlimit{4.0}
\pgfpathmoveto{\pgfqpoint{0.682cm}{0.671cm}}
\pgfpathlineto{\pgfqpoint{0.682cm}{0.042cm}}
\pgfusepath{stroke}
\end{pgfscope}
\end{pgfscope}
\end{pgfscope}
\end{pgfscope}
\end{tikzpicture}}}\circ X,\qquad X^{\!\resizebox{!}{.8em}{
\begin{tikzpicture}
\pgfpathmoveto{\pgfqpoint{0cm}{-0.035cm}}
\pgfpathlineto{\pgfqpoint{1.976cm}{-0.035cm}}
\pgfpathlineto{\pgfqpoint{1.976cm}{1.94cm}}
\pgfpathlineto{\pgfqpoint{0cm}{1.94cm}}
\pgfpathclose
\pgfusepath{clip}
\begin{pgfscope}
\begin{pgfscope}
\pgfpathmoveto{\pgfqpoint{0cm}{-0.035cm}}
\pgfpathlineto{\pgfqpoint{1.976cm}{-0.035cm}}
\pgfpathlineto{\pgfqpoint{1.976cm}{1.94cm}}
\pgfpathlineto{\pgfqpoint{0cm}{1.94cm}}
\pgfpathclose
\pgfusepath{clip}
\begin{pgfscope}
\begin{pgfscope}
\pgfsetdash{}{0cm}
\pgfsetlinewidth{0.818mm}
\pgfsetroundcap
\pgfsetroundjoin
\pgfsetmiterlimit{7.0}
\definecolor{eps2pgf_color}{gray}{0}\pgfsetstrokecolor{eps2pgf_color}\pgfsetfillcolor{eps2pgf_color}
\pgfpathmoveto{\pgfqpoint{0.117cm}{1.815cm}}
\pgfpathlineto{\pgfqpoint{0.682cm}{1.065cm}}
\pgfpathlineto{\pgfqpoint{1.246cm}{1.815cm}}
\pgfusepath{stroke}
\end{pgfscope}
\definecolor{eps2pgf_color}{gray}{0}\pgfsetstrokecolor{eps2pgf_color}\pgfsetfillcolor{eps2pgf_color}
\pgfpathmoveto{\pgfqpoint{0.273cm}{1.789cm}}
\pgfpathcurveto{\pgfqpoint{0.273cm}{1.825cm}}{\pgfqpoint{0.259cm}{1.86cm}}{\pgfqpoint{0.233cm}{1.886cm}}
\pgfpathcurveto{\pgfqpoint{0.207cm}{1.912cm}}{\pgfqpoint{0.173cm}{1.926cm}}{\pgfqpoint{0.137cm}{1.926cm}}
\pgfpathcurveto{\pgfqpoint{0.1cm}{1.926cm}}{\pgfqpoint{0.066cm}{1.912cm}}{\pgfqpoint{0.04cm}{1.886cm}}
\pgfpathcurveto{\pgfqpoint{0.014cm}{1.86cm}}{\pgfqpoint{0cm}{1.825cm}}{\pgfqpoint{0cm}{1.789cm}}
\pgfpathcurveto{\pgfqpoint{0cm}{1.753cm}}{\pgfqpoint{0.014cm}{1.718cm}}{\pgfqpoint{0.04cm}{1.692cm}}
\pgfpathcurveto{\pgfqpoint{0.066cm}{1.667cm}}{\pgfqpoint{0.1cm}{1.652cm}}{\pgfqpoint{0.137cm}{1.652cm}}
\pgfpathcurveto{\pgfqpoint{0.173cm}{1.652cm}}{\pgfqpoint{0.207cm}{1.667cm}}{\pgfqpoint{0.233cm}{1.692cm}}
\pgfpathcurveto{\pgfqpoint{0.259cm}{1.718cm}}{\pgfqpoint{0.273cm}{1.753cm}}{\pgfqpoint{0.273cm}{1.789cm}}
\pgfusepath{fill}
\pgfpathmoveto{\pgfqpoint{1.345cm}{1.765cm}}
\pgfpathcurveto{\pgfqpoint{1.345cm}{1.801cm}}{\pgfqpoint{1.331cm}{1.836cm}}{\pgfqpoint{1.305cm}{1.862cm}}
\pgfpathcurveto{\pgfqpoint{1.28cm}{1.887cm}}{\pgfqpoint{1.245cm}{1.902cm}}{\pgfqpoint{1.209cm}{1.902cm}}
\pgfpathcurveto{\pgfqpoint{1.172cm}{1.902cm}}{\pgfqpoint{1.138cm}{1.887cm}}{\pgfqpoint{1.112cm}{1.862cm}}
\pgfpathcurveto{\pgfqpoint{1.087cm}{1.836cm}}{\pgfqpoint{1.072cm}{1.801cm}}{\pgfqpoint{1.072cm}{1.765cm}}
\pgfpathcurveto{\pgfqpoint{1.072cm}{1.728cm}}{\pgfqpoint{1.087cm}{1.694cm}}{\pgfqpoint{1.112cm}{1.668cm}}
\pgfpathcurveto{\pgfqpoint{1.138cm}{1.642cm}}{\pgfqpoint{1.172cm}{1.628cm}}{\pgfqpoint{1.209cm}{1.628cm}}
\pgfpathcurveto{\pgfqpoint{1.245cm}{1.628cm}}{\pgfqpoint{1.28cm}{1.642cm}}{\pgfqpoint{1.305cm}{1.668cm}}
\pgfpathcurveto{\pgfqpoint{1.331cm}{1.694cm}}{\pgfqpoint{1.345cm}{1.728cm}}{\pgfqpoint{1.345cm}{1.765cm}}
\pgfusepath{fill}
\begin{pgfscope}
\pgfsetdash{}{0cm}
\pgfsetlinewidth{0.818mm}
\pgfsetroundcap
\pgfsetroundjoin
\pgfsetmiterlimit{7.0}
\pgfpathmoveto{\pgfqpoint{0.682cm}{1.065cm}}
\pgfpathlineto{\pgfqpoint{1.246cm}{0.315cm}}
\pgfpathlineto{\pgfqpoint{1.811cm}{1.065cm}}
\pgfusepath{stroke}
\end{pgfscope}
\pgfpathmoveto{\pgfqpoint{1.948cm}{1.065cm}}
\pgfpathcurveto{\pgfqpoint{1.948cm}{1.101cm}}{\pgfqpoint{1.933cm}{1.136cm}}{\pgfqpoint{1.907cm}{1.162cm}}
\pgfpathcurveto{\pgfqpoint{1.882cm}{1.187cm}}{\pgfqpoint{1.847cm}{1.202cm}}{\pgfqpoint{1.811cm}{1.202cm}}
\pgfpathcurveto{\pgfqpoint{1.775cm}{1.202cm}}{\pgfqpoint{1.74cm}{1.187cm}}{\pgfqpoint{1.714cm}{1.162cm}}
\pgfpathcurveto{\pgfqpoint{1.689cm}{1.136cm}}{\pgfqpoint{1.674cm}{1.101cm}}{\pgfqpoint{1.674cm}{1.065cm}}
\pgfpathcurveto{\pgfqpoint{1.674cm}{1.029cm}}{\pgfqpoint{1.689cm}{0.994cm}}{\pgfqpoint{1.714cm}{0.968cm}}
\pgfpathcurveto{\pgfqpoint{1.74cm}{0.942cm}}{\pgfqpoint{1.775cm}{0.928cm}}{\pgfqpoint{1.811cm}{0.928cm}}
\pgfpathcurveto{\pgfqpoint{1.847cm}{0.928cm}}{\pgfqpoint{1.882cm}{0.942cm}}{\pgfqpoint{1.907cm}{0.968cm}}
\pgfpathcurveto{\pgfqpoint{1.933cm}{0.994cm}}{\pgfqpoint{1.948cm}{1.029cm}}{\pgfqpoint{1.948cm}{1.065cm}}
\pgfusepath{fill}
\begin{pgfscope}
\pgfsetdash{}{0cm}
\pgfsetlinewidth{0.818mm}
\pgfsetmiterlimit{7.0}
\pgfpathmoveto{\pgfqpoint{1.246cm}{0.315cm}}
\pgfpathlineto{\pgfqpoint{1.244cm}{1.061cm}}
\pgfusepath{stroke}
\end{pgfscope}
\pgfpathmoveto{\pgfqpoint{1.38cm}{1.065cm}}
\pgfpathcurveto{\pgfqpoint{1.38cm}{1.101cm}}{\pgfqpoint{1.366cm}{1.136cm}}{\pgfqpoint{1.34cm}{1.162cm}}
\pgfpathcurveto{\pgfqpoint{1.315cm}{1.187cm}}{\pgfqpoint{1.28cm}{1.202cm}}{\pgfqpoint{1.244cm}{1.202cm}}
\pgfpathcurveto{\pgfqpoint{1.207cm}{1.202cm}}{\pgfqpoint{1.173cm}{1.187cm}}{\pgfqpoint{1.147cm}{1.162cm}}
\pgfpathcurveto{\pgfqpoint{1.121cm}{1.136cm}}{\pgfqpoint{1.107cm}{1.101cm}}{\pgfqpoint{1.107cm}{1.065cm}}
\pgfpathcurveto{\pgfqpoint{1.107cm}{1.029cm}}{\pgfqpoint{1.121cm}{0.994cm}}{\pgfqpoint{1.147cm}{0.968cm}}
\pgfpathcurveto{\pgfqpoint{1.173cm}{0.942cm}}{\pgfqpoint{1.207cm}{0.928cm}}{\pgfqpoint{1.244cm}{0.928cm}}
\pgfpathcurveto{\pgfqpoint{1.28cm}{0.928cm}}{\pgfqpoint{1.315cm}{0.942cm}}{\pgfqpoint{1.34cm}{0.968cm}}
\pgfpathcurveto{\pgfqpoint{1.366cm}{0.994cm}}{\pgfqpoint{1.38cm}{1.029cm}}{\pgfqpoint{1.38cm}{1.065cm}}
\pgfusepath{fill}
\begin{pgfscope}
\pgfsetdash{}{0cm}
\pgfsetlinewidth{0.818mm}
\pgfsetmiterlimit{4.0}
\pgfpathmoveto{\pgfqpoint{1.383cm}{0.178cm}}
\pgfpathcurveto{\pgfqpoint{1.383cm}{0.214cm}}{\pgfqpoint{1.369cm}{0.249cm}}{\pgfqpoint{1.343cm}{0.275cm}}
\pgfpathcurveto{\pgfqpoint{1.317cm}{0.3cm}}{\pgfqpoint{1.283cm}{0.315cm}}{\pgfqpoint{1.246cm}{0.315cm}}
\pgfpathcurveto{\pgfqpoint{1.21cm}{0.315cm}}{\pgfqpoint{1.175cm}{0.3cm}}{\pgfqpoint{1.15cm}{0.275cm}}
\pgfpathcurveto{\pgfqpoint{1.124cm}{0.249cm}}{\pgfqpoint{1.11cm}{0.214cm}}{\pgfqpoint{1.11cm}{0.178cm}}
\pgfpathcurveto{\pgfqpoint{1.11cm}{0.141cm}}{\pgfqpoint{1.124cm}{0.107cm}}{\pgfqpoint{1.15cm}{0.081cm}}
\pgfpathcurveto{\pgfqpoint{1.175cm}{0.055cm}}{\pgfqpoint{1.21cm}{0.041cm}}{\pgfqpoint{1.246cm}{0.041cm}}
\pgfpathcurveto{\pgfqpoint{1.283cm}{0.041cm}}{\pgfqpoint{1.317cm}{0.055cm}}{\pgfqpoint{1.343cm}{0.081cm}}
\pgfpathcurveto{\pgfqpoint{1.369cm}{0.107cm}}{\pgfqpoint{1.383cm}{0.141cm}}{\pgfqpoint{1.383cm}{0.178cm}}
\pgfusepath{stroke}
\end{pgfscope}
\end{pgfscope}
\end{pgfscope}
\end{pgfscope}
\end{tikzpicture}}} = X^{\!\resizebox{0.6em}{!}{
\begin{tikzpicture}
\pgfpathmoveto{\pgfqpoint{0cm}{0cm}}
\pgfpathlineto{\pgfqpoint{1.376cm}{0cm}}
\pgfpathlineto{\pgfqpoint{1.376cm}{1.588cm}}
\pgfpathlineto{\pgfqpoint{0cm}{1.588cm}}
\pgfpathclose
\pgfusepath{clip}
\begin{pgfscope}
\begin{pgfscope}
\pgfpathmoveto{\pgfqpoint{0cm}{0cm}}
\pgfpathlineto{\pgfqpoint{1.376cm}{0cm}}
\pgfpathlineto{\pgfqpoint{1.376cm}{1.588cm}}
\pgfpathlineto{\pgfqpoint{0cm}{1.588cm}}
\pgfpathclose
\pgfusepath{clip}
\begin{pgfscope}
\begin{pgfscope}
\definecolor{eps2pgf_color}{gray}{0.976471}\pgfsetstrokecolor{eps2pgf_color}\pgfsetfillcolor{eps2pgf_color}
\pgfpathmoveto{\pgfqpoint{0cm}{0cm}}
\pgfpathlineto{\pgfqpoint{1.376cm}{0cm}}
\pgfpathlineto{\pgfqpoint{1.376cm}{1.588cm}}
\pgfpathlineto{\pgfqpoint{0cm}{1.588cm}}
\pgfpathclose
\pgfusepath{fill}
\end{pgfscope}
\begin{pgfscope}
\pgfsetdash{}{0cm}
\pgfsetlinewidth{0.818mm}
\pgfsetroundcap
\pgfsetroundjoin
\pgfsetmiterlimit{7.0}
\definecolor{eps2pgf_color}{gray}{0}\pgfsetstrokecolor{eps2pgf_color}\pgfsetfillcolor{eps2pgf_color}
\pgfpathmoveto{\pgfqpoint{0.117cm}{1.476cm}}
\pgfpathlineto{\pgfqpoint{0.682cm}{0.726cm}}
\pgfpathlineto{\pgfqpoint{1.246cm}{1.476cm}}
\pgfusepath{stroke}
\end{pgfscope}
\definecolor{eps2pgf_color}{gray}{0}\pgfsetstrokecolor{eps2pgf_color}\pgfsetfillcolor{eps2pgf_color}
\pgfpathmoveto{\pgfqpoint{0.273cm}{1.451cm}}
\pgfpathcurveto{\pgfqpoint{0.273cm}{1.487cm}}{\pgfqpoint{0.259cm}{1.522cm}}{\pgfqpoint{0.233cm}{1.547cm}}
\pgfpathcurveto{\pgfqpoint{0.207cm}{1.573cm}}{\pgfqpoint{0.173cm}{1.588cm}}{\pgfqpoint{0.137cm}{1.588cm}}
\pgfpathcurveto{\pgfqpoint{0.1cm}{1.588cm}}{\pgfqpoint{0.066cm}{1.573cm}}{\pgfqpoint{0.04cm}{1.547cm}}
\pgfpathcurveto{\pgfqpoint{0.014cm}{1.522cm}}{\pgfqpoint{0cm}{1.487cm}}{\pgfqpoint{0cm}{1.451cm}}
\pgfpathcurveto{\pgfqpoint{0cm}{1.414cm}}{\pgfqpoint{0.014cm}{1.379cm}}{\pgfqpoint{0.04cm}{1.354cm}}
\pgfpathcurveto{\pgfqpoint{0.066cm}{1.328cm}}{\pgfqpoint{0.1cm}{1.314cm}}{\pgfqpoint{0.137cm}{1.314cm}}
\pgfpathcurveto{\pgfqpoint{0.173cm}{1.314cm}}{\pgfqpoint{0.207cm}{1.328cm}}{\pgfqpoint{0.233cm}{1.354cm}}
\pgfpathcurveto{\pgfqpoint{0.259cm}{1.379cm}}{\pgfqpoint{0.273cm}{1.414cm}}{\pgfqpoint{0.273cm}{1.451cm}}
\pgfusepath{fill}
\pgfpathmoveto{\pgfqpoint{1.345cm}{1.426cm}}
\pgfpathcurveto{\pgfqpoint{1.345cm}{1.463cm}}{\pgfqpoint{1.331cm}{1.497cm}}{\pgfqpoint{1.305cm}{1.523cm}}
\pgfpathcurveto{\pgfqpoint{1.28cm}{1.549cm}}{\pgfqpoint{1.245cm}{1.563cm}}{\pgfqpoint{1.209cm}{1.563cm}}
\pgfpathcurveto{\pgfqpoint{1.172cm}{1.563cm}}{\pgfqpoint{1.138cm}{1.549cm}}{\pgfqpoint{1.112cm}{1.523cm}}
\pgfpathcurveto{\pgfqpoint{1.087cm}{1.497cm}}{\pgfqpoint{1.072cm}{1.463cm}}{\pgfqpoint{1.072cm}{1.426cm}}
\pgfpathcurveto{\pgfqpoint{1.072cm}{1.39cm}}{\pgfqpoint{1.087cm}{1.355cm}}{\pgfqpoint{1.112cm}{1.329cm}}
\pgfpathcurveto{\pgfqpoint{1.138cm}{1.304cm}}{\pgfqpoint{1.172cm}{1.289cm}}{\pgfqpoint{1.209cm}{1.289cm}}
\pgfpathcurveto{\pgfqpoint{1.245cm}{1.289cm}}{\pgfqpoint{1.28cm}{1.304cm}}{\pgfqpoint{1.305cm}{1.329cm}}
\pgfpathcurveto{\pgfqpoint{1.331cm}{1.355cm}}{\pgfqpoint{1.345cm}{1.39cm}}{\pgfqpoint{1.345cm}{1.426cm}}
\pgfusepath{fill}
\begin{pgfscope}
\pgfsetdash{}{0cm}
\pgfsetlinewidth{0.818mm}
\pgfsetroundcap
\pgfsetmiterlimit{4.0}
\pgfpathmoveto{\pgfqpoint{0.682cm}{0.726cm}}
\pgfpathlineto{\pgfqpoint{0.682cm}{0.097cm}}
\pgfusepath{stroke}
\end{pgfscope}
\end{pgfscope}
\end{pgfscope}
\end{pgfscope}
\end{tikzpicture}}} \circ \llbracket X^2 \rrbracket -\frac{b}{3}, \qquad
   X^{\!\resizebox{!}{.8em}{
\begin{tikzpicture}
\pgfpathmoveto{\pgfqpoint{0cm}{-0.035cm}}
\pgfpathlineto{\pgfqpoint{1.976cm}{-0.035cm}}
\pgfpathlineto{\pgfqpoint{1.976cm}{1.94cm}}
\pgfpathlineto{\pgfqpoint{0cm}{1.94cm}}
\pgfpathclose
\pgfusepath{clip}
\begin{pgfscope}
\begin{pgfscope}
\pgfpathmoveto{\pgfqpoint{0cm}{-0.035cm}}
\pgfpathlineto{\pgfqpoint{1.976cm}{-0.035cm}}
\pgfpathlineto{\pgfqpoint{1.976cm}{1.94cm}}
\pgfpathlineto{\pgfqpoint{0cm}{1.94cm}}
\pgfpathclose
\pgfusepath{clip}
\begin{pgfscope}
\begin{pgfscope}
\pgfsetdash{}{0cm}
\pgfsetlinewidth{0.818mm}
\pgfsetroundcap
\pgfsetroundjoin
\pgfsetmiterlimit{7.0}
\definecolor{eps2pgf_color}{gray}{0}\pgfsetstrokecolor{eps2pgf_color}\pgfsetfillcolor{eps2pgf_color}
\pgfpathmoveto{\pgfqpoint{0.117cm}{1.815cm}}
\pgfpathlineto{\pgfqpoint{0.682cm}{1.065cm}}
\pgfpathlineto{\pgfqpoint{1.246cm}{1.815cm}}
\pgfusepath{stroke}
\end{pgfscope}
\definecolor{eps2pgf_color}{gray}{0}\pgfsetstrokecolor{eps2pgf_color}\pgfsetfillcolor{eps2pgf_color}
\pgfpathmoveto{\pgfqpoint{0.273cm}{1.789cm}}
\pgfpathcurveto{\pgfqpoint{0.273cm}{1.825cm}}{\pgfqpoint{0.259cm}{1.86cm}}{\pgfqpoint{0.233cm}{1.886cm}}
\pgfpathcurveto{\pgfqpoint{0.207cm}{1.912cm}}{\pgfqpoint{0.173cm}{1.926cm}}{\pgfqpoint{0.137cm}{1.926cm}}
\pgfpathcurveto{\pgfqpoint{0.1cm}{1.926cm}}{\pgfqpoint{0.066cm}{1.912cm}}{\pgfqpoint{0.04cm}{1.886cm}}
\pgfpathcurveto{\pgfqpoint{0.014cm}{1.86cm}}{\pgfqpoint{0cm}{1.825cm}}{\pgfqpoint{0cm}{1.789cm}}
\pgfpathcurveto{\pgfqpoint{0cm}{1.753cm}}{\pgfqpoint{0.014cm}{1.718cm}}{\pgfqpoint{0.04cm}{1.692cm}}
\pgfpathcurveto{\pgfqpoint{0.066cm}{1.667cm}}{\pgfqpoint{0.1cm}{1.652cm}}{\pgfqpoint{0.137cm}{1.652cm}}
\pgfpathcurveto{\pgfqpoint{0.173cm}{1.652cm}}{\pgfqpoint{0.207cm}{1.667cm}}{\pgfqpoint{0.233cm}{1.692cm}}
\pgfpathcurveto{\pgfqpoint{0.259cm}{1.718cm}}{\pgfqpoint{0.273cm}{1.753cm}}{\pgfqpoint{0.273cm}{1.789cm}}
\pgfusepath{fill}
\begin{pgfscope}
\pgfsetdash{}{0cm}
\pgfsetlinewidth{0.818mm}
\pgfsetmiterlimit{7.0}
\pgfpathmoveto{\pgfqpoint{0.682cm}{1.065cm}}
\pgfpathlineto{\pgfqpoint{0.679cm}{1.812cm}}
\pgfusepath{stroke}
\end{pgfscope}
\pgfpathmoveto{\pgfqpoint{0.815cm}{1.793cm}}
\pgfpathcurveto{\pgfqpoint{0.815cm}{1.829cm}}{\pgfqpoint{0.801cm}{1.864cm}}{\pgfqpoint{0.775cm}{1.89cm}}
\pgfpathcurveto{\pgfqpoint{0.75cm}{1.915cm}}{\pgfqpoint{0.715cm}{1.93cm}}{\pgfqpoint{0.679cm}{1.93cm}}
\pgfpathcurveto{\pgfqpoint{0.643cm}{1.93cm}}{\pgfqpoint{0.608cm}{1.915cm}}{\pgfqpoint{0.582cm}{1.89cm}}
\pgfpathcurveto{\pgfqpoint{0.557cm}{1.864cm}}{\pgfqpoint{0.542cm}{1.829cm}}{\pgfqpoint{0.542cm}{1.793cm}}
\pgfpathcurveto{\pgfqpoint{0.542cm}{1.756cm}}{\pgfqpoint{0.557cm}{1.722cm}}{\pgfqpoint{0.582cm}{1.696cm}}
\pgfpathcurveto{\pgfqpoint{0.608cm}{1.67cm}}{\pgfqpoint{0.643cm}{1.656cm}}{\pgfqpoint{0.679cm}{1.656cm}}
\pgfpathcurveto{\pgfqpoint{0.715cm}{1.656cm}}{\pgfqpoint{0.75cm}{1.67cm}}{\pgfqpoint{0.775cm}{1.696cm}}
\pgfpathcurveto{\pgfqpoint{0.801cm}{1.722cm}}{\pgfqpoint{0.815cm}{1.756cm}}{\pgfqpoint{0.815cm}{1.793cm}}
\pgfusepath{fill}
\pgfpathmoveto{\pgfqpoint{1.345cm}{1.765cm}}
\pgfpathcurveto{\pgfqpoint{1.345cm}{1.801cm}}{\pgfqpoint{1.331cm}{1.836cm}}{\pgfqpoint{1.305cm}{1.862cm}}
\pgfpathcurveto{\pgfqpoint{1.28cm}{1.887cm}}{\pgfqpoint{1.245cm}{1.902cm}}{\pgfqpoint{1.209cm}{1.902cm}}
\pgfpathcurveto{\pgfqpoint{1.172cm}{1.902cm}}{\pgfqpoint{1.138cm}{1.887cm}}{\pgfqpoint{1.112cm}{1.862cm}}
\pgfpathcurveto{\pgfqpoint{1.087cm}{1.836cm}}{\pgfqpoint{1.072cm}{1.801cm}}{\pgfqpoint{1.072cm}{1.765cm}}
\pgfpathcurveto{\pgfqpoint{1.072cm}{1.728cm}}{\pgfqpoint{1.087cm}{1.694cm}}{\pgfqpoint{1.112cm}{1.668cm}}
\pgfpathcurveto{\pgfqpoint{1.138cm}{1.642cm}}{\pgfqpoint{1.172cm}{1.628cm}}{\pgfqpoint{1.209cm}{1.628cm}}
\pgfpathcurveto{\pgfqpoint{1.245cm}{1.628cm}}{\pgfqpoint{1.28cm}{1.642cm}}{\pgfqpoint{1.305cm}{1.668cm}}
\pgfpathcurveto{\pgfqpoint{1.331cm}{1.694cm}}{\pgfqpoint{1.345cm}{1.728cm}}{\pgfqpoint{1.345cm}{1.765cm}}
\pgfusepath{fill}
\begin{pgfscope}
\pgfsetdash{}{0cm}
\pgfsetlinewidth{0.818mm}
\pgfsetroundcap
\pgfsetroundjoin
\pgfsetmiterlimit{7.0}
\pgfpathmoveto{\pgfqpoint{0.682cm}{1.065cm}}
\pgfpathlineto{\pgfqpoint{1.246cm}{0.315cm}}
\pgfpathlineto{\pgfqpoint{1.811cm}{1.065cm}}
\pgfusepath{stroke}
\end{pgfscope}
\pgfpathmoveto{\pgfqpoint{1.948cm}{1.065cm}}
\pgfpathcurveto{\pgfqpoint{1.948cm}{1.101cm}}{\pgfqpoint{1.933cm}{1.136cm}}{\pgfqpoint{1.907cm}{1.162cm}}
\pgfpathcurveto{\pgfqpoint{1.882cm}{1.187cm}}{\pgfqpoint{1.847cm}{1.202cm}}{\pgfqpoint{1.811cm}{1.202cm}}
\pgfpathcurveto{\pgfqpoint{1.775cm}{1.202cm}}{\pgfqpoint{1.74cm}{1.187cm}}{\pgfqpoint{1.714cm}{1.162cm}}
\pgfpathcurveto{\pgfqpoint{1.689cm}{1.136cm}}{\pgfqpoint{1.674cm}{1.101cm}}{\pgfqpoint{1.674cm}{1.065cm}}
\pgfpathcurveto{\pgfqpoint{1.674cm}{1.029cm}}{\pgfqpoint{1.689cm}{0.994cm}}{\pgfqpoint{1.714cm}{0.968cm}}
\pgfpathcurveto{\pgfqpoint{1.74cm}{0.942cm}}{\pgfqpoint{1.775cm}{0.928cm}}{\pgfqpoint{1.811cm}{0.928cm}}
\pgfpathcurveto{\pgfqpoint{1.847cm}{0.928cm}}{\pgfqpoint{1.882cm}{0.942cm}}{\pgfqpoint{1.907cm}{0.968cm}}
\pgfpathcurveto{\pgfqpoint{1.933cm}{0.994cm}}{\pgfqpoint{1.948cm}{1.029cm}}{\pgfqpoint{1.948cm}{1.065cm}}
\pgfusepath{fill}
\begin{pgfscope}
\pgfsetdash{}{0cm}
\pgfsetlinewidth{0.818mm}
\pgfsetmiterlimit{7.0}
\pgfpathmoveto{\pgfqpoint{1.246cm}{0.315cm}}
\pgfpathlineto{\pgfqpoint{1.244cm}{1.061cm}}
\pgfusepath{stroke}
\end{pgfscope}
\pgfpathmoveto{\pgfqpoint{1.38cm}{1.065cm}}
\pgfpathcurveto{\pgfqpoint{1.38cm}{1.101cm}}{\pgfqpoint{1.366cm}{1.136cm}}{\pgfqpoint{1.34cm}{1.162cm}}
\pgfpathcurveto{\pgfqpoint{1.315cm}{1.187cm}}{\pgfqpoint{1.28cm}{1.202cm}}{\pgfqpoint{1.244cm}{1.202cm}}
\pgfpathcurveto{\pgfqpoint{1.207cm}{1.202cm}}{\pgfqpoint{1.173cm}{1.187cm}}{\pgfqpoint{1.147cm}{1.162cm}}
\pgfpathcurveto{\pgfqpoint{1.121cm}{1.136cm}}{\pgfqpoint{1.107cm}{1.101cm}}{\pgfqpoint{1.107cm}{1.065cm}}
\pgfpathcurveto{\pgfqpoint{1.107cm}{1.029cm}}{\pgfqpoint{1.121cm}{0.994cm}}{\pgfqpoint{1.147cm}{0.968cm}}
\pgfpathcurveto{\pgfqpoint{1.173cm}{0.942cm}}{\pgfqpoint{1.207cm}{0.928cm}}{\pgfqpoint{1.244cm}{0.928cm}}
\pgfpathcurveto{\pgfqpoint{1.28cm}{0.928cm}}{\pgfqpoint{1.315cm}{0.942cm}}{\pgfqpoint{1.34cm}{0.968cm}}
\pgfpathcurveto{\pgfqpoint{1.366cm}{0.994cm}}{\pgfqpoint{1.38cm}{1.029cm}}{\pgfqpoint{1.38cm}{1.065cm}}
\pgfusepath{fill}
\begin{pgfscope}
\pgfsetdash{}{0cm}
\pgfsetlinewidth{0.818mm}
\pgfsetmiterlimit{4.0}
\pgfpathmoveto{\pgfqpoint{1.383cm}{0.178cm}}
\pgfpathcurveto{\pgfqpoint{1.383cm}{0.214cm}}{\pgfqpoint{1.369cm}{0.249cm}}{\pgfqpoint{1.343cm}{0.275cm}}
\pgfpathcurveto{\pgfqpoint{1.317cm}{0.3cm}}{\pgfqpoint{1.283cm}{0.315cm}}{\pgfqpoint{1.246cm}{0.315cm}}
\pgfpathcurveto{\pgfqpoint{1.21cm}{0.315cm}}{\pgfqpoint{1.175cm}{0.3cm}}{\pgfqpoint{1.15cm}{0.275cm}}
\pgfpathcurveto{\pgfqpoint{1.124cm}{0.249cm}}{\pgfqpoint{1.11cm}{0.214cm}}{\pgfqpoint{1.11cm}{0.178cm}}
\pgfpathcurveto{\pgfqpoint{1.11cm}{0.141cm}}{\pgfqpoint{1.124cm}{0.107cm}}{\pgfqpoint{1.15cm}{0.081cm}}
\pgfpathcurveto{\pgfqpoint{1.175cm}{0.055cm}}{\pgfqpoint{1.21cm}{0.041cm}}{\pgfqpoint{1.246cm}{0.041cm}}
\pgfpathcurveto{\pgfqpoint{1.283cm}{0.041cm}}{\pgfqpoint{1.317cm}{0.055cm}}{\pgfqpoint{1.343cm}{0.081cm}}
\pgfpathcurveto{\pgfqpoint{1.369cm}{0.107cm}}{\pgfqpoint{1.383cm}{0.141cm}}{\pgfqpoint{1.383cm}{0.178cm}}
\pgfusepath{stroke}
\end{pgfscope}
\end{pgfscope}
\end{pgfscope}
\end{pgfscope}
\end{tikzpicture}}} = X^{\!\resizebox{0.6em}{!}{
\begin{tikzpicture}
\pgfpathmoveto{\pgfqpoint{0cm}{-0.035cm}}
\pgfpathlineto{\pgfqpoint{1.376cm}{-0.035cm}}
\pgfpathlineto{\pgfqpoint{1.376cm}{1.552cm}}
\pgfpathlineto{\pgfqpoint{0cm}{1.552cm}}
\pgfpathclose
\pgfusepath{clip}
\begin{pgfscope}
\begin{pgfscope}
\pgfpathmoveto{\pgfqpoint{0cm}{-0.035cm}}
\pgfpathlineto{\pgfqpoint{1.376cm}{-0.035cm}}
\pgfpathlineto{\pgfqpoint{1.376cm}{1.552cm}}
\pgfpathlineto{\pgfqpoint{0cm}{1.552cm}}
\pgfpathclose
\pgfusepath{clip}
\begin{pgfscope}
\begin{pgfscope}
\pgfsetdash{}{0cm}
\pgfsetlinewidth{0.818mm}
\pgfsetroundcap
\pgfsetroundjoin
\pgfsetmiterlimit{7.0}
\definecolor{eps2pgf_color}{gray}{0}\pgfsetstrokecolor{eps2pgf_color}\pgfsetfillcolor{eps2pgf_color}
\pgfpathmoveto{\pgfqpoint{0.117cm}{1.421cm}}
\pgfpathlineto{\pgfqpoint{0.682cm}{0.671cm}}
\pgfpathlineto{\pgfqpoint{1.246cm}{1.421cm}}
\pgfusepath{stroke}
\end{pgfscope}
\definecolor{eps2pgf_color}{gray}{0}\pgfsetstrokecolor{eps2pgf_color}\pgfsetfillcolor{eps2pgf_color}
\pgfpathmoveto{\pgfqpoint{0.273cm}{1.395cm}}
\pgfpathcurveto{\pgfqpoint{0.273cm}{1.432cm}}{\pgfqpoint{0.259cm}{1.467cm}}{\pgfqpoint{0.233cm}{1.492cm}}
\pgfpathcurveto{\pgfqpoint{0.207cm}{1.518cm}}{\pgfqpoint{0.173cm}{1.532cm}}{\pgfqpoint{0.137cm}{1.532cm}}
\pgfpathcurveto{\pgfqpoint{0.1cm}{1.532cm}}{\pgfqpoint{0.066cm}{1.518cm}}{\pgfqpoint{0.04cm}{1.492cm}}
\pgfpathcurveto{\pgfqpoint{0.014cm}{1.467cm}}{\pgfqpoint{0cm}{1.432cm}}{\pgfqpoint{0cm}{1.395cm}}
\pgfpathcurveto{\pgfqpoint{0cm}{1.359cm}}{\pgfqpoint{0.014cm}{1.324cm}}{\pgfqpoint{0.04cm}{1.299cm}}
\pgfpathcurveto{\pgfqpoint{0.066cm}{1.273cm}}{\pgfqpoint{0.1cm}{1.258cm}}{\pgfqpoint{0.137cm}{1.258cm}}
\pgfpathcurveto{\pgfqpoint{0.173cm}{1.258cm}}{\pgfqpoint{0.207cm}{1.273cm}}{\pgfqpoint{0.233cm}{1.299cm}}
\pgfpathcurveto{\pgfqpoint{0.259cm}{1.324cm}}{\pgfqpoint{0.273cm}{1.359cm}}{\pgfqpoint{0.273cm}{1.395cm}}
\pgfusepath{fill}
\begin{pgfscope}
\pgfsetdash{}{0cm}
\pgfsetlinewidth{0.818mm}
\pgfsetmiterlimit{7.0}
\pgfpathmoveto{\pgfqpoint{0.682cm}{0.671cm}}
\pgfpathlineto{\pgfqpoint{0.679cm}{1.418cm}}
\pgfusepath{stroke}
\end{pgfscope}
\pgfpathmoveto{\pgfqpoint{0.815cm}{1.399cm}}
\pgfpathcurveto{\pgfqpoint{0.815cm}{1.435cm}}{\pgfqpoint{0.801cm}{1.47cm}}{\pgfqpoint{0.775cm}{1.496cm}}
\pgfpathcurveto{\pgfqpoint{0.75cm}{1.521cm}}{\pgfqpoint{0.715cm}{1.536cm}}{\pgfqpoint{0.679cm}{1.536cm}}
\pgfpathcurveto{\pgfqpoint{0.643cm}{1.536cm}}{\pgfqpoint{0.608cm}{1.521cm}}{\pgfqpoint{0.582cm}{1.496cm}}
\pgfpathcurveto{\pgfqpoint{0.557cm}{1.47cm}}{\pgfqpoint{0.542cm}{1.435cm}}{\pgfqpoint{0.542cm}{1.399cm}}
\pgfpathcurveto{\pgfqpoint{0.542cm}{1.363cm}}{\pgfqpoint{0.557cm}{1.328cm}}{\pgfqpoint{0.582cm}{1.302cm}}
\pgfpathcurveto{\pgfqpoint{0.608cm}{1.276cm}}{\pgfqpoint{0.643cm}{1.262cm}}{\pgfqpoint{0.679cm}{1.262cm}}
\pgfpathcurveto{\pgfqpoint{0.715cm}{1.262cm}}{\pgfqpoint{0.75cm}{1.276cm}}{\pgfqpoint{0.775cm}{1.302cm}}
\pgfpathcurveto{\pgfqpoint{0.801cm}{1.328cm}}{\pgfqpoint{0.815cm}{1.363cm}}{\pgfqpoint{0.815cm}{1.399cm}}
\pgfusepath{fill}
\pgfpathmoveto{\pgfqpoint{1.345cm}{1.371cm}}
\pgfpathcurveto{\pgfqpoint{1.345cm}{1.408cm}}{\pgfqpoint{1.331cm}{1.442cm}}{\pgfqpoint{1.305cm}{1.468cm}}
\pgfpathcurveto{\pgfqpoint{1.28cm}{1.494cm}}{\pgfqpoint{1.245cm}{1.508cm}}{\pgfqpoint{1.209cm}{1.508cm}}
\pgfpathcurveto{\pgfqpoint{1.172cm}{1.508cm}}{\pgfqpoint{1.138cm}{1.494cm}}{\pgfqpoint{1.112cm}{1.468cm}}
\pgfpathcurveto{\pgfqpoint{1.087cm}{1.442cm}}{\pgfqpoint{1.072cm}{1.408cm}}{\pgfqpoint{1.072cm}{1.371cm}}
\pgfpathcurveto{\pgfqpoint{1.072cm}{1.335cm}}{\pgfqpoint{1.087cm}{1.3cm}}{\pgfqpoint{1.112cm}{1.274cm}}
\pgfpathcurveto{\pgfqpoint{1.138cm}{1.249cm}}{\pgfqpoint{1.172cm}{1.234cm}}{\pgfqpoint{1.209cm}{1.234cm}}
\pgfpathcurveto{\pgfqpoint{1.245cm}{1.234cm}}{\pgfqpoint{1.28cm}{1.249cm}}{\pgfqpoint{1.305cm}{1.274cm}}
\pgfpathcurveto{\pgfqpoint{1.331cm}{1.3cm}}{\pgfqpoint{1.345cm}{1.335cm}}{\pgfqpoint{1.345cm}{1.371cm}}
\pgfusepath{fill}
\begin{pgfscope}
\pgfsetdash{}{0cm}
\pgfsetlinewidth{0.818mm}
\pgfsetroundcap
\pgfsetmiterlimit{4.0}
\pgfpathmoveto{\pgfqpoint{0.682cm}{0.671cm}}
\pgfpathlineto{\pgfqpoint{0.682cm}{0.042cm}}
\pgfusepath{stroke}
\end{pgfscope}
\end{pgfscope}
\end{pgfscope}
\end{pgfscope}
\end{tikzpicture}}} \circ \llbracket X^2 \rrbracket -  bX. \]

Recall that if $\rho$ is a polynomial weight of the form $\rho(x)=\langle x\rangle^{-\nu}$ for some   $\nu>0$ and $\sigma>0$ then these objects can be constructed in spaces $\CC^{\alpha}(\rho^\sigma)$ where the respective values of $\alpha$ are given in  Table~\ref{t:reg}. The parameter $\kappa>0$ can be chosen arbitrarily small. 

As a consequence, the left hand side of \eqref{eq:45} rewrites as
\begin{align}\label{eq:rhs45}
\begin{aligned}
&\Q \varphi + \varphi^3 + (- 3 a + 3b) \varphi - \xi \\
&= \Q\phi+\Q\psi +3\llbracket X^2 \rrbracket(-X^{\!\resizebox{0.6em}{!}{
\begin{tikzpicture}
\pgfpathmoveto{\pgfqpoint{0cm}{-0.035cm}}
\pgfpathlineto{\pgfqpoint{1.376cm}{-0.035cm}}
\pgfpathlineto{\pgfqpoint{1.376cm}{1.552cm}}
\pgfpathlineto{\pgfqpoint{0cm}{1.552cm}}
\pgfpathclose
\pgfusepath{clip}
\begin{pgfscope}
\begin{pgfscope}
\pgfpathmoveto{\pgfqpoint{0cm}{-0.035cm}}
\pgfpathlineto{\pgfqpoint{1.376cm}{-0.035cm}}
\pgfpathlineto{\pgfqpoint{1.376cm}{1.552cm}}
\pgfpathlineto{\pgfqpoint{0cm}{1.552cm}}
\pgfpathclose
\pgfusepath{clip}
\begin{pgfscope}
\begin{pgfscope}
\pgfsetdash{}{0cm}
\pgfsetlinewidth{0.818mm}
\pgfsetroundcap
\pgfsetroundjoin
\pgfsetmiterlimit{7.0}
\definecolor{eps2pgf_color}{gray}{0}\pgfsetstrokecolor{eps2pgf_color}\pgfsetfillcolor{eps2pgf_color}
\pgfpathmoveto{\pgfqpoint{0.117cm}{1.421cm}}
\pgfpathlineto{\pgfqpoint{0.682cm}{0.671cm}}
\pgfpathlineto{\pgfqpoint{1.246cm}{1.421cm}}
\pgfusepath{stroke}
\end{pgfscope}
\definecolor{eps2pgf_color}{gray}{0}\pgfsetstrokecolor{eps2pgf_color}\pgfsetfillcolor{eps2pgf_color}
\pgfpathmoveto{\pgfqpoint{0.273cm}{1.395cm}}
\pgfpathcurveto{\pgfqpoint{0.273cm}{1.432cm}}{\pgfqpoint{0.259cm}{1.467cm}}{\pgfqpoint{0.233cm}{1.492cm}}
\pgfpathcurveto{\pgfqpoint{0.207cm}{1.518cm}}{\pgfqpoint{0.173cm}{1.532cm}}{\pgfqpoint{0.137cm}{1.532cm}}
\pgfpathcurveto{\pgfqpoint{0.1cm}{1.532cm}}{\pgfqpoint{0.066cm}{1.518cm}}{\pgfqpoint{0.04cm}{1.492cm}}
\pgfpathcurveto{\pgfqpoint{0.014cm}{1.467cm}}{\pgfqpoint{0cm}{1.432cm}}{\pgfqpoint{0cm}{1.395cm}}
\pgfpathcurveto{\pgfqpoint{0cm}{1.359cm}}{\pgfqpoint{0.014cm}{1.324cm}}{\pgfqpoint{0.04cm}{1.299cm}}
\pgfpathcurveto{\pgfqpoint{0.066cm}{1.273cm}}{\pgfqpoint{0.1cm}{1.258cm}}{\pgfqpoint{0.137cm}{1.258cm}}
\pgfpathcurveto{\pgfqpoint{0.173cm}{1.258cm}}{\pgfqpoint{0.207cm}{1.273cm}}{\pgfqpoint{0.233cm}{1.299cm}}
\pgfpathcurveto{\pgfqpoint{0.259cm}{1.324cm}}{\pgfqpoint{0.273cm}{1.359cm}}{\pgfqpoint{0.273cm}{1.395cm}}
\pgfusepath{fill}
\begin{pgfscope}
\pgfsetdash{}{0cm}
\pgfsetlinewidth{0.818mm}
\pgfsetmiterlimit{7.0}
\pgfpathmoveto{\pgfqpoint{0.682cm}{0.671cm}}
\pgfpathlineto{\pgfqpoint{0.679cm}{1.418cm}}
\pgfusepath{stroke}
\end{pgfscope}
\pgfpathmoveto{\pgfqpoint{0.815cm}{1.399cm}}
\pgfpathcurveto{\pgfqpoint{0.815cm}{1.435cm}}{\pgfqpoint{0.801cm}{1.47cm}}{\pgfqpoint{0.775cm}{1.496cm}}
\pgfpathcurveto{\pgfqpoint{0.75cm}{1.521cm}}{\pgfqpoint{0.715cm}{1.536cm}}{\pgfqpoint{0.679cm}{1.536cm}}
\pgfpathcurveto{\pgfqpoint{0.643cm}{1.536cm}}{\pgfqpoint{0.608cm}{1.521cm}}{\pgfqpoint{0.582cm}{1.496cm}}
\pgfpathcurveto{\pgfqpoint{0.557cm}{1.47cm}}{\pgfqpoint{0.542cm}{1.435cm}}{\pgfqpoint{0.542cm}{1.399cm}}
\pgfpathcurveto{\pgfqpoint{0.542cm}{1.363cm}}{\pgfqpoint{0.557cm}{1.328cm}}{\pgfqpoint{0.582cm}{1.302cm}}
\pgfpathcurveto{\pgfqpoint{0.608cm}{1.276cm}}{\pgfqpoint{0.643cm}{1.262cm}}{\pgfqpoint{0.679cm}{1.262cm}}
\pgfpathcurveto{\pgfqpoint{0.715cm}{1.262cm}}{\pgfqpoint{0.75cm}{1.276cm}}{\pgfqpoint{0.775cm}{1.302cm}}
\pgfpathcurveto{\pgfqpoint{0.801cm}{1.328cm}}{\pgfqpoint{0.815cm}{1.363cm}}{\pgfqpoint{0.815cm}{1.399cm}}
\pgfusepath{fill}
\pgfpathmoveto{\pgfqpoint{1.345cm}{1.371cm}}
\pgfpathcurveto{\pgfqpoint{1.345cm}{1.408cm}}{\pgfqpoint{1.331cm}{1.442cm}}{\pgfqpoint{1.305cm}{1.468cm}}
\pgfpathcurveto{\pgfqpoint{1.28cm}{1.494cm}}{\pgfqpoint{1.245cm}{1.508cm}}{\pgfqpoint{1.209cm}{1.508cm}}
\pgfpathcurveto{\pgfqpoint{1.172cm}{1.508cm}}{\pgfqpoint{1.138cm}{1.494cm}}{\pgfqpoint{1.112cm}{1.468cm}}
\pgfpathcurveto{\pgfqpoint{1.087cm}{1.442cm}}{\pgfqpoint{1.072cm}{1.408cm}}{\pgfqpoint{1.072cm}{1.371cm}}
\pgfpathcurveto{\pgfqpoint{1.072cm}{1.335cm}}{\pgfqpoint{1.087cm}{1.3cm}}{\pgfqpoint{1.112cm}{1.274cm}}
\pgfpathcurveto{\pgfqpoint{1.138cm}{1.249cm}}{\pgfqpoint{1.172cm}{1.234cm}}{\pgfqpoint{1.209cm}{1.234cm}}
\pgfpathcurveto{\pgfqpoint{1.245cm}{1.234cm}}{\pgfqpoint{1.28cm}{1.249cm}}{\pgfqpoint{1.305cm}{1.274cm}}
\pgfpathcurveto{\pgfqpoint{1.331cm}{1.3cm}}{\pgfqpoint{1.345cm}{1.335cm}}{\pgfqpoint{1.345cm}{1.371cm}}
\pgfusepath{fill}
\begin{pgfscope}
\pgfsetdash{}{0cm}
\pgfsetlinewidth{0.818mm}
\pgfsetroundcap
\pgfsetmiterlimit{4.0}
\pgfpathmoveto{\pgfqpoint{0.682cm}{0.671cm}}
\pgfpathlineto{\pgfqpoint{0.682cm}{0.042cm}}
\pgfusepath{stroke}
\end{pgfscope}
\end{pgfscope}
\end{pgfscope}
\end{pgfscope}
\end{tikzpicture}}}+\phi+\psi)+3X(-X^{\!\resizebox{0.6em}{!}{
\begin{tikzpicture}
\pgfpathmoveto{\pgfqpoint{0cm}{-0.035cm}}
\pgfpathlineto{\pgfqpoint{1.376cm}{-0.035cm}}
\pgfpathlineto{\pgfqpoint{1.376cm}{1.552cm}}
\pgfpathlineto{\pgfqpoint{0cm}{1.552cm}}
\pgfpathclose
\pgfusepath{clip}
\begin{pgfscope}
\begin{pgfscope}
\pgfpathmoveto{\pgfqpoint{0cm}{-0.035cm}}
\pgfpathlineto{\pgfqpoint{1.376cm}{-0.035cm}}
\pgfpathlineto{\pgfqpoint{1.376cm}{1.552cm}}
\pgfpathlineto{\pgfqpoint{0cm}{1.552cm}}
\pgfpathclose
\pgfusepath{clip}
\begin{pgfscope}
\begin{pgfscope}
\pgfsetdash{}{0cm}
\pgfsetlinewidth{0.818mm}
\pgfsetroundcap
\pgfsetroundjoin
\pgfsetmiterlimit{7.0}
\definecolor{eps2pgf_color}{gray}{0}\pgfsetstrokecolor{eps2pgf_color}\pgfsetfillcolor{eps2pgf_color}
\pgfpathmoveto{\pgfqpoint{0.117cm}{1.421cm}}
\pgfpathlineto{\pgfqpoint{0.682cm}{0.671cm}}
\pgfpathlineto{\pgfqpoint{1.246cm}{1.421cm}}
\pgfusepath{stroke}
\end{pgfscope}
\definecolor{eps2pgf_color}{gray}{0}\pgfsetstrokecolor{eps2pgf_color}\pgfsetfillcolor{eps2pgf_color}
\pgfpathmoveto{\pgfqpoint{0.273cm}{1.395cm}}
\pgfpathcurveto{\pgfqpoint{0.273cm}{1.432cm}}{\pgfqpoint{0.259cm}{1.467cm}}{\pgfqpoint{0.233cm}{1.492cm}}
\pgfpathcurveto{\pgfqpoint{0.207cm}{1.518cm}}{\pgfqpoint{0.173cm}{1.532cm}}{\pgfqpoint{0.137cm}{1.532cm}}
\pgfpathcurveto{\pgfqpoint{0.1cm}{1.532cm}}{\pgfqpoint{0.066cm}{1.518cm}}{\pgfqpoint{0.04cm}{1.492cm}}
\pgfpathcurveto{\pgfqpoint{0.014cm}{1.467cm}}{\pgfqpoint{0cm}{1.432cm}}{\pgfqpoint{0cm}{1.395cm}}
\pgfpathcurveto{\pgfqpoint{0cm}{1.359cm}}{\pgfqpoint{0.014cm}{1.324cm}}{\pgfqpoint{0.04cm}{1.299cm}}
\pgfpathcurveto{\pgfqpoint{0.066cm}{1.273cm}}{\pgfqpoint{0.1cm}{1.258cm}}{\pgfqpoint{0.137cm}{1.258cm}}
\pgfpathcurveto{\pgfqpoint{0.173cm}{1.258cm}}{\pgfqpoint{0.207cm}{1.273cm}}{\pgfqpoint{0.233cm}{1.299cm}}
\pgfpathcurveto{\pgfqpoint{0.259cm}{1.324cm}}{\pgfqpoint{0.273cm}{1.359cm}}{\pgfqpoint{0.273cm}{1.395cm}}
\pgfusepath{fill}
\begin{pgfscope}
\pgfsetdash{}{0cm}
\pgfsetlinewidth{0.818mm}
\pgfsetmiterlimit{7.0}
\pgfpathmoveto{\pgfqpoint{0.682cm}{0.671cm}}
\pgfpathlineto{\pgfqpoint{0.679cm}{1.418cm}}
\pgfusepath{stroke}
\end{pgfscope}
\pgfpathmoveto{\pgfqpoint{0.815cm}{1.399cm}}
\pgfpathcurveto{\pgfqpoint{0.815cm}{1.435cm}}{\pgfqpoint{0.801cm}{1.47cm}}{\pgfqpoint{0.775cm}{1.496cm}}
\pgfpathcurveto{\pgfqpoint{0.75cm}{1.521cm}}{\pgfqpoint{0.715cm}{1.536cm}}{\pgfqpoint{0.679cm}{1.536cm}}
\pgfpathcurveto{\pgfqpoint{0.643cm}{1.536cm}}{\pgfqpoint{0.608cm}{1.521cm}}{\pgfqpoint{0.582cm}{1.496cm}}
\pgfpathcurveto{\pgfqpoint{0.557cm}{1.47cm}}{\pgfqpoint{0.542cm}{1.435cm}}{\pgfqpoint{0.542cm}{1.399cm}}
\pgfpathcurveto{\pgfqpoint{0.542cm}{1.363cm}}{\pgfqpoint{0.557cm}{1.328cm}}{\pgfqpoint{0.582cm}{1.302cm}}
\pgfpathcurveto{\pgfqpoint{0.608cm}{1.276cm}}{\pgfqpoint{0.643cm}{1.262cm}}{\pgfqpoint{0.679cm}{1.262cm}}
\pgfpathcurveto{\pgfqpoint{0.715cm}{1.262cm}}{\pgfqpoint{0.75cm}{1.276cm}}{\pgfqpoint{0.775cm}{1.302cm}}
\pgfpathcurveto{\pgfqpoint{0.801cm}{1.328cm}}{\pgfqpoint{0.815cm}{1.363cm}}{\pgfqpoint{0.815cm}{1.399cm}}
\pgfusepath{fill}
\pgfpathmoveto{\pgfqpoint{1.345cm}{1.371cm}}
\pgfpathcurveto{\pgfqpoint{1.345cm}{1.408cm}}{\pgfqpoint{1.331cm}{1.442cm}}{\pgfqpoint{1.305cm}{1.468cm}}
\pgfpathcurveto{\pgfqpoint{1.28cm}{1.494cm}}{\pgfqpoint{1.245cm}{1.508cm}}{\pgfqpoint{1.209cm}{1.508cm}}
\pgfpathcurveto{\pgfqpoint{1.172cm}{1.508cm}}{\pgfqpoint{1.138cm}{1.494cm}}{\pgfqpoint{1.112cm}{1.468cm}}
\pgfpathcurveto{\pgfqpoint{1.087cm}{1.442cm}}{\pgfqpoint{1.072cm}{1.408cm}}{\pgfqpoint{1.072cm}{1.371cm}}
\pgfpathcurveto{\pgfqpoint{1.072cm}{1.335cm}}{\pgfqpoint{1.087cm}{1.3cm}}{\pgfqpoint{1.112cm}{1.274cm}}
\pgfpathcurveto{\pgfqpoint{1.138cm}{1.249cm}}{\pgfqpoint{1.172cm}{1.234cm}}{\pgfqpoint{1.209cm}{1.234cm}}
\pgfpathcurveto{\pgfqpoint{1.245cm}{1.234cm}}{\pgfqpoint{1.28cm}{1.249cm}}{\pgfqpoint{1.305cm}{1.274cm}}
\pgfpathcurveto{\pgfqpoint{1.331cm}{1.3cm}}{\pgfqpoint{1.345cm}{1.335cm}}{\pgfqpoint{1.345cm}{1.371cm}}
\pgfusepath{fill}
\begin{pgfscope}
\pgfsetdash{}{0cm}
\pgfsetlinewidth{0.818mm}
\pgfsetroundcap
\pgfsetmiterlimit{4.0}
\pgfpathmoveto{\pgfqpoint{0.682cm}{0.671cm}}
\pgfpathlineto{\pgfqpoint{0.682cm}{0.042cm}}
\pgfusepath{stroke}
\end{pgfscope}
\end{pgfscope}
\end{pgfscope}
\end{pgfscope}
\end{tikzpicture}}}+\phi+\psi)^2+(-X^{\!\resizebox{0.6em}{!}{
\begin{tikzpicture}
\pgfpathmoveto{\pgfqpoint{0cm}{-0.035cm}}
\pgfpathlineto{\pgfqpoint{1.376cm}{-0.035cm}}
\pgfpathlineto{\pgfqpoint{1.376cm}{1.552cm}}
\pgfpathlineto{\pgfqpoint{0cm}{1.552cm}}
\pgfpathclose
\pgfusepath{clip}
\begin{pgfscope}
\begin{pgfscope}
\pgfpathmoveto{\pgfqpoint{0cm}{-0.035cm}}
\pgfpathlineto{\pgfqpoint{1.376cm}{-0.035cm}}
\pgfpathlineto{\pgfqpoint{1.376cm}{1.552cm}}
\pgfpathlineto{\pgfqpoint{0cm}{1.552cm}}
\pgfpathclose
\pgfusepath{clip}
\begin{pgfscope}
\begin{pgfscope}
\pgfsetdash{}{0cm}
\pgfsetlinewidth{0.818mm}
\pgfsetroundcap
\pgfsetroundjoin
\pgfsetmiterlimit{7.0}
\definecolor{eps2pgf_color}{gray}{0}\pgfsetstrokecolor{eps2pgf_color}\pgfsetfillcolor{eps2pgf_color}
\pgfpathmoveto{\pgfqpoint{0.117cm}{1.421cm}}
\pgfpathlineto{\pgfqpoint{0.682cm}{0.671cm}}
\pgfpathlineto{\pgfqpoint{1.246cm}{1.421cm}}
\pgfusepath{stroke}
\end{pgfscope}
\definecolor{eps2pgf_color}{gray}{0}\pgfsetstrokecolor{eps2pgf_color}\pgfsetfillcolor{eps2pgf_color}
\pgfpathmoveto{\pgfqpoint{0.273cm}{1.395cm}}
\pgfpathcurveto{\pgfqpoint{0.273cm}{1.432cm}}{\pgfqpoint{0.259cm}{1.467cm}}{\pgfqpoint{0.233cm}{1.492cm}}
\pgfpathcurveto{\pgfqpoint{0.207cm}{1.518cm}}{\pgfqpoint{0.173cm}{1.532cm}}{\pgfqpoint{0.137cm}{1.532cm}}
\pgfpathcurveto{\pgfqpoint{0.1cm}{1.532cm}}{\pgfqpoint{0.066cm}{1.518cm}}{\pgfqpoint{0.04cm}{1.492cm}}
\pgfpathcurveto{\pgfqpoint{0.014cm}{1.467cm}}{\pgfqpoint{0cm}{1.432cm}}{\pgfqpoint{0cm}{1.395cm}}
\pgfpathcurveto{\pgfqpoint{0cm}{1.359cm}}{\pgfqpoint{0.014cm}{1.324cm}}{\pgfqpoint{0.04cm}{1.299cm}}
\pgfpathcurveto{\pgfqpoint{0.066cm}{1.273cm}}{\pgfqpoint{0.1cm}{1.258cm}}{\pgfqpoint{0.137cm}{1.258cm}}
\pgfpathcurveto{\pgfqpoint{0.173cm}{1.258cm}}{\pgfqpoint{0.207cm}{1.273cm}}{\pgfqpoint{0.233cm}{1.299cm}}
\pgfpathcurveto{\pgfqpoint{0.259cm}{1.324cm}}{\pgfqpoint{0.273cm}{1.359cm}}{\pgfqpoint{0.273cm}{1.395cm}}
\pgfusepath{fill}
\begin{pgfscope}
\pgfsetdash{}{0cm}
\pgfsetlinewidth{0.818mm}
\pgfsetmiterlimit{7.0}
\pgfpathmoveto{\pgfqpoint{0.682cm}{0.671cm}}
\pgfpathlineto{\pgfqpoint{0.679cm}{1.418cm}}
\pgfusepath{stroke}
\end{pgfscope}
\pgfpathmoveto{\pgfqpoint{0.815cm}{1.399cm}}
\pgfpathcurveto{\pgfqpoint{0.815cm}{1.435cm}}{\pgfqpoint{0.801cm}{1.47cm}}{\pgfqpoint{0.775cm}{1.496cm}}
\pgfpathcurveto{\pgfqpoint{0.75cm}{1.521cm}}{\pgfqpoint{0.715cm}{1.536cm}}{\pgfqpoint{0.679cm}{1.536cm}}
\pgfpathcurveto{\pgfqpoint{0.643cm}{1.536cm}}{\pgfqpoint{0.608cm}{1.521cm}}{\pgfqpoint{0.582cm}{1.496cm}}
\pgfpathcurveto{\pgfqpoint{0.557cm}{1.47cm}}{\pgfqpoint{0.542cm}{1.435cm}}{\pgfqpoint{0.542cm}{1.399cm}}
\pgfpathcurveto{\pgfqpoint{0.542cm}{1.363cm}}{\pgfqpoint{0.557cm}{1.328cm}}{\pgfqpoint{0.582cm}{1.302cm}}
\pgfpathcurveto{\pgfqpoint{0.608cm}{1.276cm}}{\pgfqpoint{0.643cm}{1.262cm}}{\pgfqpoint{0.679cm}{1.262cm}}
\pgfpathcurveto{\pgfqpoint{0.715cm}{1.262cm}}{\pgfqpoint{0.75cm}{1.276cm}}{\pgfqpoint{0.775cm}{1.302cm}}
\pgfpathcurveto{\pgfqpoint{0.801cm}{1.328cm}}{\pgfqpoint{0.815cm}{1.363cm}}{\pgfqpoint{0.815cm}{1.399cm}}
\pgfusepath{fill}
\pgfpathmoveto{\pgfqpoint{1.345cm}{1.371cm}}
\pgfpathcurveto{\pgfqpoint{1.345cm}{1.408cm}}{\pgfqpoint{1.331cm}{1.442cm}}{\pgfqpoint{1.305cm}{1.468cm}}
\pgfpathcurveto{\pgfqpoint{1.28cm}{1.494cm}}{\pgfqpoint{1.245cm}{1.508cm}}{\pgfqpoint{1.209cm}{1.508cm}}
\pgfpathcurveto{\pgfqpoint{1.172cm}{1.508cm}}{\pgfqpoint{1.138cm}{1.494cm}}{\pgfqpoint{1.112cm}{1.468cm}}
\pgfpathcurveto{\pgfqpoint{1.087cm}{1.442cm}}{\pgfqpoint{1.072cm}{1.408cm}}{\pgfqpoint{1.072cm}{1.371cm}}
\pgfpathcurveto{\pgfqpoint{1.072cm}{1.335cm}}{\pgfqpoint{1.087cm}{1.3cm}}{\pgfqpoint{1.112cm}{1.274cm}}
\pgfpathcurveto{\pgfqpoint{1.138cm}{1.249cm}}{\pgfqpoint{1.172cm}{1.234cm}}{\pgfqpoint{1.209cm}{1.234cm}}
\pgfpathcurveto{\pgfqpoint{1.245cm}{1.234cm}}{\pgfqpoint{1.28cm}{1.249cm}}{\pgfqpoint{1.305cm}{1.274cm}}
\pgfpathcurveto{\pgfqpoint{1.331cm}{1.3cm}}{\pgfqpoint{1.345cm}{1.335cm}}{\pgfqpoint{1.345cm}{1.371cm}}
\pgfusepath{fill}
\begin{pgfscope}
\pgfsetdash{}{0cm}
\pgfsetlinewidth{0.818mm}
\pgfsetroundcap
\pgfsetmiterlimit{4.0}
\pgfpathmoveto{\pgfqpoint{0.682cm}{0.671cm}}
\pgfpathlineto{\pgfqpoint{0.682cm}{0.042cm}}
\pgfusepath{stroke}
\end{pgfscope}
\end{pgfscope}
\end{pgfscope}
\end{pgfscope}
\end{tikzpicture}}}+\phi+\psi)^3+3b\varphi.
\end{aligned}
\end{align}
Our goal is to construct $\psi$ with regularity $2+\alpha$ whereas $\phi$ will be of regularity $\frac{1}{2}+\alpha$ for some $\alpha>0$ small. Consequently, the third  term on the right hand side of \eqref{eq:rhs45} is not expected to be well-defined and difficulties also arise in the fourth term. In order to cancel the most irregular part of the third term, we assume further that $\phi$ is paracontrolled by $X^{\!\resizebox{0.6em}{!}{
\begin{tikzpicture}
\pgfpathmoveto{\pgfqpoint{0cm}{0cm}}
\pgfpathlineto{\pgfqpoint{1.376cm}{0cm}}
\pgfpathlineto{\pgfqpoint{1.376cm}{1.588cm}}
\pgfpathlineto{\pgfqpoint{0cm}{1.588cm}}
\pgfpathclose
\pgfusepath{clip}
\begin{pgfscope}
\begin{pgfscope}
\pgfpathmoveto{\pgfqpoint{0cm}{0cm}}
\pgfpathlineto{\pgfqpoint{1.376cm}{0cm}}
\pgfpathlineto{\pgfqpoint{1.376cm}{1.588cm}}
\pgfpathlineto{\pgfqpoint{0cm}{1.588cm}}
\pgfpathclose
\pgfusepath{clip}
\begin{pgfscope}
\begin{pgfscope}
\definecolor{eps2pgf_color}{gray}{0.976471}\pgfsetstrokecolor{eps2pgf_color}\pgfsetfillcolor{eps2pgf_color}
\pgfpathmoveto{\pgfqpoint{0cm}{0cm}}
\pgfpathlineto{\pgfqpoint{1.376cm}{0cm}}
\pgfpathlineto{\pgfqpoint{1.376cm}{1.588cm}}
\pgfpathlineto{\pgfqpoint{0cm}{1.588cm}}
\pgfpathclose
\pgfusepath{fill}
\end{pgfscope}
\begin{pgfscope}
\pgfsetdash{}{0cm}
\pgfsetlinewidth{0.818mm}
\pgfsetroundcap
\pgfsetroundjoin
\pgfsetmiterlimit{7.0}
\definecolor{eps2pgf_color}{gray}{0}\pgfsetstrokecolor{eps2pgf_color}\pgfsetfillcolor{eps2pgf_color}
\pgfpathmoveto{\pgfqpoint{0.117cm}{1.476cm}}
\pgfpathlineto{\pgfqpoint{0.682cm}{0.726cm}}
\pgfpathlineto{\pgfqpoint{1.246cm}{1.476cm}}
\pgfusepath{stroke}
\end{pgfscope}
\definecolor{eps2pgf_color}{gray}{0}\pgfsetstrokecolor{eps2pgf_color}\pgfsetfillcolor{eps2pgf_color}
\pgfpathmoveto{\pgfqpoint{0.273cm}{1.451cm}}
\pgfpathcurveto{\pgfqpoint{0.273cm}{1.487cm}}{\pgfqpoint{0.259cm}{1.522cm}}{\pgfqpoint{0.233cm}{1.547cm}}
\pgfpathcurveto{\pgfqpoint{0.207cm}{1.573cm}}{\pgfqpoint{0.173cm}{1.588cm}}{\pgfqpoint{0.137cm}{1.588cm}}
\pgfpathcurveto{\pgfqpoint{0.1cm}{1.588cm}}{\pgfqpoint{0.066cm}{1.573cm}}{\pgfqpoint{0.04cm}{1.547cm}}
\pgfpathcurveto{\pgfqpoint{0.014cm}{1.522cm}}{\pgfqpoint{0cm}{1.487cm}}{\pgfqpoint{0cm}{1.451cm}}
\pgfpathcurveto{\pgfqpoint{0cm}{1.414cm}}{\pgfqpoint{0.014cm}{1.379cm}}{\pgfqpoint{0.04cm}{1.354cm}}
\pgfpathcurveto{\pgfqpoint{0.066cm}{1.328cm}}{\pgfqpoint{0.1cm}{1.314cm}}{\pgfqpoint{0.137cm}{1.314cm}}
\pgfpathcurveto{\pgfqpoint{0.173cm}{1.314cm}}{\pgfqpoint{0.207cm}{1.328cm}}{\pgfqpoint{0.233cm}{1.354cm}}
\pgfpathcurveto{\pgfqpoint{0.259cm}{1.379cm}}{\pgfqpoint{0.273cm}{1.414cm}}{\pgfqpoint{0.273cm}{1.451cm}}
\pgfusepath{fill}
\pgfpathmoveto{\pgfqpoint{1.345cm}{1.426cm}}
\pgfpathcurveto{\pgfqpoint{1.345cm}{1.463cm}}{\pgfqpoint{1.331cm}{1.497cm}}{\pgfqpoint{1.305cm}{1.523cm}}
\pgfpathcurveto{\pgfqpoint{1.28cm}{1.549cm}}{\pgfqpoint{1.245cm}{1.563cm}}{\pgfqpoint{1.209cm}{1.563cm}}
\pgfpathcurveto{\pgfqpoint{1.172cm}{1.563cm}}{\pgfqpoint{1.138cm}{1.549cm}}{\pgfqpoint{1.112cm}{1.523cm}}
\pgfpathcurveto{\pgfqpoint{1.087cm}{1.497cm}}{\pgfqpoint{1.072cm}{1.463cm}}{\pgfqpoint{1.072cm}{1.426cm}}
\pgfpathcurveto{\pgfqpoint{1.072cm}{1.39cm}}{\pgfqpoint{1.087cm}{1.355cm}}{\pgfqpoint{1.112cm}{1.329cm}}
\pgfpathcurveto{\pgfqpoint{1.138cm}{1.304cm}}{\pgfqpoint{1.172cm}{1.289cm}}{\pgfqpoint{1.209cm}{1.289cm}}
\pgfpathcurveto{\pgfqpoint{1.245cm}{1.289cm}}{\pgfqpoint{1.28cm}{1.304cm}}{\pgfqpoint{1.305cm}{1.329cm}}
\pgfpathcurveto{\pgfqpoint{1.331cm}{1.355cm}}{\pgfqpoint{1.345cm}{1.39cm}}{\pgfqpoint{1.345cm}{1.426cm}}
\pgfusepath{fill}
\begin{pgfscope}
\pgfsetdash{}{0cm}
\pgfsetlinewidth{0.818mm}
\pgfsetroundcap
\pgfsetmiterlimit{4.0}
\pgfpathmoveto{\pgfqpoint{0.682cm}{0.726cm}}
\pgfpathlineto{\pgfqpoint{0.682cm}{0.097cm}}
\pgfusepath{stroke}
\end{pgfscope}
\end{pgfscope}
\end{pgfscope}
\end{pgfscope}
\end{tikzpicture}}}$, namely, it holds
\begin{align}\label{eq:th}
\phi=\vartheta -3(-X^{\!\resizebox{0.6em}{!}{
\begin{tikzpicture}
\pgfpathmoveto{\pgfqpoint{0cm}{-0.035cm}}
\pgfpathlineto{\pgfqpoint{1.376cm}{-0.035cm}}
\pgfpathlineto{\pgfqpoint{1.376cm}{1.552cm}}
\pgfpathlineto{\pgfqpoint{0cm}{1.552cm}}
\pgfpathclose
\pgfusepath{clip}
\begin{pgfscope}
\begin{pgfscope}
\pgfpathmoveto{\pgfqpoint{0cm}{-0.035cm}}
\pgfpathlineto{\pgfqpoint{1.376cm}{-0.035cm}}
\pgfpathlineto{\pgfqpoint{1.376cm}{1.552cm}}
\pgfpathlineto{\pgfqpoint{0cm}{1.552cm}}
\pgfpathclose
\pgfusepath{clip}
\begin{pgfscope}
\begin{pgfscope}
\pgfsetdash{}{0cm}
\pgfsetlinewidth{0.818mm}
\pgfsetroundcap
\pgfsetroundjoin
\pgfsetmiterlimit{7.0}
\definecolor{eps2pgf_color}{gray}{0}\pgfsetstrokecolor{eps2pgf_color}\pgfsetfillcolor{eps2pgf_color}
\pgfpathmoveto{\pgfqpoint{0.117cm}{1.421cm}}
\pgfpathlineto{\pgfqpoint{0.682cm}{0.671cm}}
\pgfpathlineto{\pgfqpoint{1.246cm}{1.421cm}}
\pgfusepath{stroke}
\end{pgfscope}
\definecolor{eps2pgf_color}{gray}{0}\pgfsetstrokecolor{eps2pgf_color}\pgfsetfillcolor{eps2pgf_color}
\pgfpathmoveto{\pgfqpoint{0.273cm}{1.395cm}}
\pgfpathcurveto{\pgfqpoint{0.273cm}{1.432cm}}{\pgfqpoint{0.259cm}{1.467cm}}{\pgfqpoint{0.233cm}{1.492cm}}
\pgfpathcurveto{\pgfqpoint{0.207cm}{1.518cm}}{\pgfqpoint{0.173cm}{1.532cm}}{\pgfqpoint{0.137cm}{1.532cm}}
\pgfpathcurveto{\pgfqpoint{0.1cm}{1.532cm}}{\pgfqpoint{0.066cm}{1.518cm}}{\pgfqpoint{0.04cm}{1.492cm}}
\pgfpathcurveto{\pgfqpoint{0.014cm}{1.467cm}}{\pgfqpoint{0cm}{1.432cm}}{\pgfqpoint{0cm}{1.395cm}}
\pgfpathcurveto{\pgfqpoint{0cm}{1.359cm}}{\pgfqpoint{0.014cm}{1.324cm}}{\pgfqpoint{0.04cm}{1.299cm}}
\pgfpathcurveto{\pgfqpoint{0.066cm}{1.273cm}}{\pgfqpoint{0.1cm}{1.258cm}}{\pgfqpoint{0.137cm}{1.258cm}}
\pgfpathcurveto{\pgfqpoint{0.173cm}{1.258cm}}{\pgfqpoint{0.207cm}{1.273cm}}{\pgfqpoint{0.233cm}{1.299cm}}
\pgfpathcurveto{\pgfqpoint{0.259cm}{1.324cm}}{\pgfqpoint{0.273cm}{1.359cm}}{\pgfqpoint{0.273cm}{1.395cm}}
\pgfusepath{fill}
\begin{pgfscope}
\pgfsetdash{}{0cm}
\pgfsetlinewidth{0.818mm}
\pgfsetmiterlimit{7.0}
\pgfpathmoveto{\pgfqpoint{0.682cm}{0.671cm}}
\pgfpathlineto{\pgfqpoint{0.679cm}{1.418cm}}
\pgfusepath{stroke}
\end{pgfscope}
\pgfpathmoveto{\pgfqpoint{0.815cm}{1.399cm}}
\pgfpathcurveto{\pgfqpoint{0.815cm}{1.435cm}}{\pgfqpoint{0.801cm}{1.47cm}}{\pgfqpoint{0.775cm}{1.496cm}}
\pgfpathcurveto{\pgfqpoint{0.75cm}{1.521cm}}{\pgfqpoint{0.715cm}{1.536cm}}{\pgfqpoint{0.679cm}{1.536cm}}
\pgfpathcurveto{\pgfqpoint{0.643cm}{1.536cm}}{\pgfqpoint{0.608cm}{1.521cm}}{\pgfqpoint{0.582cm}{1.496cm}}
\pgfpathcurveto{\pgfqpoint{0.557cm}{1.47cm}}{\pgfqpoint{0.542cm}{1.435cm}}{\pgfqpoint{0.542cm}{1.399cm}}
\pgfpathcurveto{\pgfqpoint{0.542cm}{1.363cm}}{\pgfqpoint{0.557cm}{1.328cm}}{\pgfqpoint{0.582cm}{1.302cm}}
\pgfpathcurveto{\pgfqpoint{0.608cm}{1.276cm}}{\pgfqpoint{0.643cm}{1.262cm}}{\pgfqpoint{0.679cm}{1.262cm}}
\pgfpathcurveto{\pgfqpoint{0.715cm}{1.262cm}}{\pgfqpoint{0.75cm}{1.276cm}}{\pgfqpoint{0.775cm}{1.302cm}}
\pgfpathcurveto{\pgfqpoint{0.801cm}{1.328cm}}{\pgfqpoint{0.815cm}{1.363cm}}{\pgfqpoint{0.815cm}{1.399cm}}
\pgfusepath{fill}
\pgfpathmoveto{\pgfqpoint{1.345cm}{1.371cm}}
\pgfpathcurveto{\pgfqpoint{1.345cm}{1.408cm}}{\pgfqpoint{1.331cm}{1.442cm}}{\pgfqpoint{1.305cm}{1.468cm}}
\pgfpathcurveto{\pgfqpoint{1.28cm}{1.494cm}}{\pgfqpoint{1.245cm}{1.508cm}}{\pgfqpoint{1.209cm}{1.508cm}}
\pgfpathcurveto{\pgfqpoint{1.172cm}{1.508cm}}{\pgfqpoint{1.138cm}{1.494cm}}{\pgfqpoint{1.112cm}{1.468cm}}
\pgfpathcurveto{\pgfqpoint{1.087cm}{1.442cm}}{\pgfqpoint{1.072cm}{1.408cm}}{\pgfqpoint{1.072cm}{1.371cm}}
\pgfpathcurveto{\pgfqpoint{1.072cm}{1.335cm}}{\pgfqpoint{1.087cm}{1.3cm}}{\pgfqpoint{1.112cm}{1.274cm}}
\pgfpathcurveto{\pgfqpoint{1.138cm}{1.249cm}}{\pgfqpoint{1.172cm}{1.234cm}}{\pgfqpoint{1.209cm}{1.234cm}}
\pgfpathcurveto{\pgfqpoint{1.245cm}{1.234cm}}{\pgfqpoint{1.28cm}{1.249cm}}{\pgfqpoint{1.305cm}{1.274cm}}
\pgfpathcurveto{\pgfqpoint{1.331cm}{1.3cm}}{\pgfqpoint{1.345cm}{1.335cm}}{\pgfqpoint{1.345cm}{1.371cm}}
\pgfusepath{fill}
\begin{pgfscope}
\pgfsetdash{}{0cm}
\pgfsetlinewidth{0.818mm}
\pgfsetroundcap
\pgfsetmiterlimit{4.0}
\pgfpathmoveto{\pgfqpoint{0.682cm}{0.671cm}}
\pgfpathlineto{\pgfqpoint{0.682cm}{0.042cm}}
\pgfusepath{stroke}
\end{pgfscope}
\end{pgfscope}
\end{pgfscope}
\end{pgfscope}
\end{tikzpicture}}} + \phi + \psi)\prec X^{\!\resizebox{0.6em}{!}{
\begin{tikzpicture}
\pgfpathmoveto{\pgfqpoint{0cm}{0cm}}
\pgfpathlineto{\pgfqpoint{1.376cm}{0cm}}
\pgfpathlineto{\pgfqpoint{1.376cm}{1.588cm}}
\pgfpathlineto{\pgfqpoint{0cm}{1.588cm}}
\pgfpathclose
\pgfusepath{clip}
\begin{pgfscope}
\begin{pgfscope}
\pgfpathmoveto{\pgfqpoint{0cm}{0cm}}
\pgfpathlineto{\pgfqpoint{1.376cm}{0cm}}
\pgfpathlineto{\pgfqpoint{1.376cm}{1.588cm}}
\pgfpathlineto{\pgfqpoint{0cm}{1.588cm}}
\pgfpathclose
\pgfusepath{clip}
\begin{pgfscope}
\begin{pgfscope}
\definecolor{eps2pgf_color}{gray}{0.976471}\pgfsetstrokecolor{eps2pgf_color}\pgfsetfillcolor{eps2pgf_color}
\pgfpathmoveto{\pgfqpoint{0cm}{0cm}}
\pgfpathlineto{\pgfqpoint{1.376cm}{0cm}}
\pgfpathlineto{\pgfqpoint{1.376cm}{1.588cm}}
\pgfpathlineto{\pgfqpoint{0cm}{1.588cm}}
\pgfpathclose
\pgfusepath{fill}
\end{pgfscope}
\begin{pgfscope}
\pgfsetdash{}{0cm}
\pgfsetlinewidth{0.818mm}
\pgfsetroundcap
\pgfsetroundjoin
\pgfsetmiterlimit{7.0}
\definecolor{eps2pgf_color}{gray}{0}\pgfsetstrokecolor{eps2pgf_color}\pgfsetfillcolor{eps2pgf_color}
\pgfpathmoveto{\pgfqpoint{0.117cm}{1.476cm}}
\pgfpathlineto{\pgfqpoint{0.682cm}{0.726cm}}
\pgfpathlineto{\pgfqpoint{1.246cm}{1.476cm}}
\pgfusepath{stroke}
\end{pgfscope}
\definecolor{eps2pgf_color}{gray}{0}\pgfsetstrokecolor{eps2pgf_color}\pgfsetfillcolor{eps2pgf_color}
\pgfpathmoveto{\pgfqpoint{0.273cm}{1.451cm}}
\pgfpathcurveto{\pgfqpoint{0.273cm}{1.487cm}}{\pgfqpoint{0.259cm}{1.522cm}}{\pgfqpoint{0.233cm}{1.547cm}}
\pgfpathcurveto{\pgfqpoint{0.207cm}{1.573cm}}{\pgfqpoint{0.173cm}{1.588cm}}{\pgfqpoint{0.137cm}{1.588cm}}
\pgfpathcurveto{\pgfqpoint{0.1cm}{1.588cm}}{\pgfqpoint{0.066cm}{1.573cm}}{\pgfqpoint{0.04cm}{1.547cm}}
\pgfpathcurveto{\pgfqpoint{0.014cm}{1.522cm}}{\pgfqpoint{0cm}{1.487cm}}{\pgfqpoint{0cm}{1.451cm}}
\pgfpathcurveto{\pgfqpoint{0cm}{1.414cm}}{\pgfqpoint{0.014cm}{1.379cm}}{\pgfqpoint{0.04cm}{1.354cm}}
\pgfpathcurveto{\pgfqpoint{0.066cm}{1.328cm}}{\pgfqpoint{0.1cm}{1.314cm}}{\pgfqpoint{0.137cm}{1.314cm}}
\pgfpathcurveto{\pgfqpoint{0.173cm}{1.314cm}}{\pgfqpoint{0.207cm}{1.328cm}}{\pgfqpoint{0.233cm}{1.354cm}}
\pgfpathcurveto{\pgfqpoint{0.259cm}{1.379cm}}{\pgfqpoint{0.273cm}{1.414cm}}{\pgfqpoint{0.273cm}{1.451cm}}
\pgfusepath{fill}
\pgfpathmoveto{\pgfqpoint{1.345cm}{1.426cm}}
\pgfpathcurveto{\pgfqpoint{1.345cm}{1.463cm}}{\pgfqpoint{1.331cm}{1.497cm}}{\pgfqpoint{1.305cm}{1.523cm}}
\pgfpathcurveto{\pgfqpoint{1.28cm}{1.549cm}}{\pgfqpoint{1.245cm}{1.563cm}}{\pgfqpoint{1.209cm}{1.563cm}}
\pgfpathcurveto{\pgfqpoint{1.172cm}{1.563cm}}{\pgfqpoint{1.138cm}{1.549cm}}{\pgfqpoint{1.112cm}{1.523cm}}
\pgfpathcurveto{\pgfqpoint{1.087cm}{1.497cm}}{\pgfqpoint{1.072cm}{1.463cm}}{\pgfqpoint{1.072cm}{1.426cm}}
\pgfpathcurveto{\pgfqpoint{1.072cm}{1.39cm}}{\pgfqpoint{1.087cm}{1.355cm}}{\pgfqpoint{1.112cm}{1.329cm}}
\pgfpathcurveto{\pgfqpoint{1.138cm}{1.304cm}}{\pgfqpoint{1.172cm}{1.289cm}}{\pgfqpoint{1.209cm}{1.289cm}}
\pgfpathcurveto{\pgfqpoint{1.245cm}{1.289cm}}{\pgfqpoint{1.28cm}{1.304cm}}{\pgfqpoint{1.305cm}{1.329cm}}
\pgfpathcurveto{\pgfqpoint{1.331cm}{1.355cm}}{\pgfqpoint{1.345cm}{1.39cm}}{\pgfqpoint{1.345cm}{1.426cm}}
\pgfusepath{fill}
\begin{pgfscope}
\pgfsetdash{}{0cm}
\pgfsetlinewidth{0.818mm}
\pgfsetroundcap
\pgfsetmiterlimit{4.0}
\pgfpathmoveto{\pgfqpoint{0.682cm}{0.726cm}}
\pgfpathlineto{\pgfqpoint{0.682cm}{0.097cm}}
\pgfusepath{stroke}
\end{pgfscope}
\end{pgfscope}
\end{pgfscope}
\end{pgfscope}
\end{tikzpicture}}}
\end{align}
for some $\vartheta$ which is more regular (we will see below that $\vartheta$ has the regularity $1+\alpha$). Hence, \eqref{eq:45} rewrites as
\begin{align}\label{eq:rhs45a}
\begin{aligned}
0&= \Q\vartheta+\Q\psi +3\llbracket X^2 \rrbracket\preccurlyeq(-X^{\!\resizebox{0.6em}{!}{
\begin{tikzpicture}
\pgfpathmoveto{\pgfqpoint{0cm}{-0.035cm}}
\pgfpathlineto{\pgfqpoint{1.376cm}{-0.035cm}}
\pgfpathlineto{\pgfqpoint{1.376cm}{1.552cm}}
\pgfpathlineto{\pgfqpoint{0cm}{1.552cm}}
\pgfpathclose
\pgfusepath{clip}
\begin{pgfscope}
\begin{pgfscope}
\pgfpathmoveto{\pgfqpoint{0cm}{-0.035cm}}
\pgfpathlineto{\pgfqpoint{1.376cm}{-0.035cm}}
\pgfpathlineto{\pgfqpoint{1.376cm}{1.552cm}}
\pgfpathlineto{\pgfqpoint{0cm}{1.552cm}}
\pgfpathclose
\pgfusepath{clip}
\begin{pgfscope}
\begin{pgfscope}
\pgfsetdash{}{0cm}
\pgfsetlinewidth{0.818mm}
\pgfsetroundcap
\pgfsetroundjoin
\pgfsetmiterlimit{7.0}
\definecolor{eps2pgf_color}{gray}{0}\pgfsetstrokecolor{eps2pgf_color}\pgfsetfillcolor{eps2pgf_color}
\pgfpathmoveto{\pgfqpoint{0.117cm}{1.421cm}}
\pgfpathlineto{\pgfqpoint{0.682cm}{0.671cm}}
\pgfpathlineto{\pgfqpoint{1.246cm}{1.421cm}}
\pgfusepath{stroke}
\end{pgfscope}
\definecolor{eps2pgf_color}{gray}{0}\pgfsetstrokecolor{eps2pgf_color}\pgfsetfillcolor{eps2pgf_color}
\pgfpathmoveto{\pgfqpoint{0.273cm}{1.395cm}}
\pgfpathcurveto{\pgfqpoint{0.273cm}{1.432cm}}{\pgfqpoint{0.259cm}{1.467cm}}{\pgfqpoint{0.233cm}{1.492cm}}
\pgfpathcurveto{\pgfqpoint{0.207cm}{1.518cm}}{\pgfqpoint{0.173cm}{1.532cm}}{\pgfqpoint{0.137cm}{1.532cm}}
\pgfpathcurveto{\pgfqpoint{0.1cm}{1.532cm}}{\pgfqpoint{0.066cm}{1.518cm}}{\pgfqpoint{0.04cm}{1.492cm}}
\pgfpathcurveto{\pgfqpoint{0.014cm}{1.467cm}}{\pgfqpoint{0cm}{1.432cm}}{\pgfqpoint{0cm}{1.395cm}}
\pgfpathcurveto{\pgfqpoint{0cm}{1.359cm}}{\pgfqpoint{0.014cm}{1.324cm}}{\pgfqpoint{0.04cm}{1.299cm}}
\pgfpathcurveto{\pgfqpoint{0.066cm}{1.273cm}}{\pgfqpoint{0.1cm}{1.258cm}}{\pgfqpoint{0.137cm}{1.258cm}}
\pgfpathcurveto{\pgfqpoint{0.173cm}{1.258cm}}{\pgfqpoint{0.207cm}{1.273cm}}{\pgfqpoint{0.233cm}{1.299cm}}
\pgfpathcurveto{\pgfqpoint{0.259cm}{1.324cm}}{\pgfqpoint{0.273cm}{1.359cm}}{\pgfqpoint{0.273cm}{1.395cm}}
\pgfusepath{fill}
\begin{pgfscope}
\pgfsetdash{}{0cm}
\pgfsetlinewidth{0.818mm}
\pgfsetmiterlimit{7.0}
\pgfpathmoveto{\pgfqpoint{0.682cm}{0.671cm}}
\pgfpathlineto{\pgfqpoint{0.679cm}{1.418cm}}
\pgfusepath{stroke}
\end{pgfscope}
\pgfpathmoveto{\pgfqpoint{0.815cm}{1.399cm}}
\pgfpathcurveto{\pgfqpoint{0.815cm}{1.435cm}}{\pgfqpoint{0.801cm}{1.47cm}}{\pgfqpoint{0.775cm}{1.496cm}}
\pgfpathcurveto{\pgfqpoint{0.75cm}{1.521cm}}{\pgfqpoint{0.715cm}{1.536cm}}{\pgfqpoint{0.679cm}{1.536cm}}
\pgfpathcurveto{\pgfqpoint{0.643cm}{1.536cm}}{\pgfqpoint{0.608cm}{1.521cm}}{\pgfqpoint{0.582cm}{1.496cm}}
\pgfpathcurveto{\pgfqpoint{0.557cm}{1.47cm}}{\pgfqpoint{0.542cm}{1.435cm}}{\pgfqpoint{0.542cm}{1.399cm}}
\pgfpathcurveto{\pgfqpoint{0.542cm}{1.363cm}}{\pgfqpoint{0.557cm}{1.328cm}}{\pgfqpoint{0.582cm}{1.302cm}}
\pgfpathcurveto{\pgfqpoint{0.608cm}{1.276cm}}{\pgfqpoint{0.643cm}{1.262cm}}{\pgfqpoint{0.679cm}{1.262cm}}
\pgfpathcurveto{\pgfqpoint{0.715cm}{1.262cm}}{\pgfqpoint{0.75cm}{1.276cm}}{\pgfqpoint{0.775cm}{1.302cm}}
\pgfpathcurveto{\pgfqpoint{0.801cm}{1.328cm}}{\pgfqpoint{0.815cm}{1.363cm}}{\pgfqpoint{0.815cm}{1.399cm}}
\pgfusepath{fill}
\pgfpathmoveto{\pgfqpoint{1.345cm}{1.371cm}}
\pgfpathcurveto{\pgfqpoint{1.345cm}{1.408cm}}{\pgfqpoint{1.331cm}{1.442cm}}{\pgfqpoint{1.305cm}{1.468cm}}
\pgfpathcurveto{\pgfqpoint{1.28cm}{1.494cm}}{\pgfqpoint{1.245cm}{1.508cm}}{\pgfqpoint{1.209cm}{1.508cm}}
\pgfpathcurveto{\pgfqpoint{1.172cm}{1.508cm}}{\pgfqpoint{1.138cm}{1.494cm}}{\pgfqpoint{1.112cm}{1.468cm}}
\pgfpathcurveto{\pgfqpoint{1.087cm}{1.442cm}}{\pgfqpoint{1.072cm}{1.408cm}}{\pgfqpoint{1.072cm}{1.371cm}}
\pgfpathcurveto{\pgfqpoint{1.072cm}{1.335cm}}{\pgfqpoint{1.087cm}{1.3cm}}{\pgfqpoint{1.112cm}{1.274cm}}
\pgfpathcurveto{\pgfqpoint{1.138cm}{1.249cm}}{\pgfqpoint{1.172cm}{1.234cm}}{\pgfqpoint{1.209cm}{1.234cm}}
\pgfpathcurveto{\pgfqpoint{1.245cm}{1.234cm}}{\pgfqpoint{1.28cm}{1.249cm}}{\pgfqpoint{1.305cm}{1.274cm}}
\pgfpathcurveto{\pgfqpoint{1.331cm}{1.3cm}}{\pgfqpoint{1.345cm}{1.335cm}}{\pgfqpoint{1.345cm}{1.371cm}}
\pgfusepath{fill}
\begin{pgfscope}
\pgfsetdash{}{0cm}
\pgfsetlinewidth{0.818mm}
\pgfsetroundcap
\pgfsetmiterlimit{4.0}
\pgfpathmoveto{\pgfqpoint{0.682cm}{0.671cm}}
\pgfpathlineto{\pgfqpoint{0.682cm}{0.042cm}}
\pgfusepath{stroke}
\end{pgfscope}
\end{pgfscope}
\end{pgfscope}
\end{pgfscope}
\end{tikzpicture}}}+\phi+\psi)-3[\Q,(-X^{\!\resizebox{0.6em}{!}{
\begin{tikzpicture}
\pgfpathmoveto{\pgfqpoint{0cm}{-0.035cm}}
\pgfpathlineto{\pgfqpoint{1.376cm}{-0.035cm}}
\pgfpathlineto{\pgfqpoint{1.376cm}{1.552cm}}
\pgfpathlineto{\pgfqpoint{0cm}{1.552cm}}
\pgfpathclose
\pgfusepath{clip}
\begin{pgfscope}
\begin{pgfscope}
\pgfpathmoveto{\pgfqpoint{0cm}{-0.035cm}}
\pgfpathlineto{\pgfqpoint{1.376cm}{-0.035cm}}
\pgfpathlineto{\pgfqpoint{1.376cm}{1.552cm}}
\pgfpathlineto{\pgfqpoint{0cm}{1.552cm}}
\pgfpathclose
\pgfusepath{clip}
\begin{pgfscope}
\begin{pgfscope}
\pgfsetdash{}{0cm}
\pgfsetlinewidth{0.818mm}
\pgfsetroundcap
\pgfsetroundjoin
\pgfsetmiterlimit{7.0}
\definecolor{eps2pgf_color}{gray}{0}\pgfsetstrokecolor{eps2pgf_color}\pgfsetfillcolor{eps2pgf_color}
\pgfpathmoveto{\pgfqpoint{0.117cm}{1.421cm}}
\pgfpathlineto{\pgfqpoint{0.682cm}{0.671cm}}
\pgfpathlineto{\pgfqpoint{1.246cm}{1.421cm}}
\pgfusepath{stroke}
\end{pgfscope}
\definecolor{eps2pgf_color}{gray}{0}\pgfsetstrokecolor{eps2pgf_color}\pgfsetfillcolor{eps2pgf_color}
\pgfpathmoveto{\pgfqpoint{0.273cm}{1.395cm}}
\pgfpathcurveto{\pgfqpoint{0.273cm}{1.432cm}}{\pgfqpoint{0.259cm}{1.467cm}}{\pgfqpoint{0.233cm}{1.492cm}}
\pgfpathcurveto{\pgfqpoint{0.207cm}{1.518cm}}{\pgfqpoint{0.173cm}{1.532cm}}{\pgfqpoint{0.137cm}{1.532cm}}
\pgfpathcurveto{\pgfqpoint{0.1cm}{1.532cm}}{\pgfqpoint{0.066cm}{1.518cm}}{\pgfqpoint{0.04cm}{1.492cm}}
\pgfpathcurveto{\pgfqpoint{0.014cm}{1.467cm}}{\pgfqpoint{0cm}{1.432cm}}{\pgfqpoint{0cm}{1.395cm}}
\pgfpathcurveto{\pgfqpoint{0cm}{1.359cm}}{\pgfqpoint{0.014cm}{1.324cm}}{\pgfqpoint{0.04cm}{1.299cm}}
\pgfpathcurveto{\pgfqpoint{0.066cm}{1.273cm}}{\pgfqpoint{0.1cm}{1.258cm}}{\pgfqpoint{0.137cm}{1.258cm}}
\pgfpathcurveto{\pgfqpoint{0.173cm}{1.258cm}}{\pgfqpoint{0.207cm}{1.273cm}}{\pgfqpoint{0.233cm}{1.299cm}}
\pgfpathcurveto{\pgfqpoint{0.259cm}{1.324cm}}{\pgfqpoint{0.273cm}{1.359cm}}{\pgfqpoint{0.273cm}{1.395cm}}
\pgfusepath{fill}
\begin{pgfscope}
\pgfsetdash{}{0cm}
\pgfsetlinewidth{0.818mm}
\pgfsetmiterlimit{7.0}
\pgfpathmoveto{\pgfqpoint{0.682cm}{0.671cm}}
\pgfpathlineto{\pgfqpoint{0.679cm}{1.418cm}}
\pgfusepath{stroke}
\end{pgfscope}
\pgfpathmoveto{\pgfqpoint{0.815cm}{1.399cm}}
\pgfpathcurveto{\pgfqpoint{0.815cm}{1.435cm}}{\pgfqpoint{0.801cm}{1.47cm}}{\pgfqpoint{0.775cm}{1.496cm}}
\pgfpathcurveto{\pgfqpoint{0.75cm}{1.521cm}}{\pgfqpoint{0.715cm}{1.536cm}}{\pgfqpoint{0.679cm}{1.536cm}}
\pgfpathcurveto{\pgfqpoint{0.643cm}{1.536cm}}{\pgfqpoint{0.608cm}{1.521cm}}{\pgfqpoint{0.582cm}{1.496cm}}
\pgfpathcurveto{\pgfqpoint{0.557cm}{1.47cm}}{\pgfqpoint{0.542cm}{1.435cm}}{\pgfqpoint{0.542cm}{1.399cm}}
\pgfpathcurveto{\pgfqpoint{0.542cm}{1.363cm}}{\pgfqpoint{0.557cm}{1.328cm}}{\pgfqpoint{0.582cm}{1.302cm}}
\pgfpathcurveto{\pgfqpoint{0.608cm}{1.276cm}}{\pgfqpoint{0.643cm}{1.262cm}}{\pgfqpoint{0.679cm}{1.262cm}}
\pgfpathcurveto{\pgfqpoint{0.715cm}{1.262cm}}{\pgfqpoint{0.75cm}{1.276cm}}{\pgfqpoint{0.775cm}{1.302cm}}
\pgfpathcurveto{\pgfqpoint{0.801cm}{1.328cm}}{\pgfqpoint{0.815cm}{1.363cm}}{\pgfqpoint{0.815cm}{1.399cm}}
\pgfusepath{fill}
\pgfpathmoveto{\pgfqpoint{1.345cm}{1.371cm}}
\pgfpathcurveto{\pgfqpoint{1.345cm}{1.408cm}}{\pgfqpoint{1.331cm}{1.442cm}}{\pgfqpoint{1.305cm}{1.468cm}}
\pgfpathcurveto{\pgfqpoint{1.28cm}{1.494cm}}{\pgfqpoint{1.245cm}{1.508cm}}{\pgfqpoint{1.209cm}{1.508cm}}
\pgfpathcurveto{\pgfqpoint{1.172cm}{1.508cm}}{\pgfqpoint{1.138cm}{1.494cm}}{\pgfqpoint{1.112cm}{1.468cm}}
\pgfpathcurveto{\pgfqpoint{1.087cm}{1.442cm}}{\pgfqpoint{1.072cm}{1.408cm}}{\pgfqpoint{1.072cm}{1.371cm}}
\pgfpathcurveto{\pgfqpoint{1.072cm}{1.335cm}}{\pgfqpoint{1.087cm}{1.3cm}}{\pgfqpoint{1.112cm}{1.274cm}}
\pgfpathcurveto{\pgfqpoint{1.138cm}{1.249cm}}{\pgfqpoint{1.172cm}{1.234cm}}{\pgfqpoint{1.209cm}{1.234cm}}
\pgfpathcurveto{\pgfqpoint{1.245cm}{1.234cm}}{\pgfqpoint{1.28cm}{1.249cm}}{\pgfqpoint{1.305cm}{1.274cm}}
\pgfpathcurveto{\pgfqpoint{1.331cm}{1.3cm}}{\pgfqpoint{1.345cm}{1.335cm}}{\pgfqpoint{1.345cm}{1.371cm}}
\pgfusepath{fill}
\begin{pgfscope}
\pgfsetdash{}{0cm}
\pgfsetlinewidth{0.818mm}
\pgfsetroundcap
\pgfsetmiterlimit{4.0}
\pgfpathmoveto{\pgfqpoint{0.682cm}{0.671cm}}
\pgfpathlineto{\pgfqpoint{0.682cm}{0.042cm}}
\pgfusepath{stroke}
\end{pgfscope}
\end{pgfscope}
\end{pgfscope}
\end{pgfscope}
\end{tikzpicture}}}+\phi+\psi)\prec]X^{\!\resizebox{0.6em}{!}{
\begin{tikzpicture}
\pgfpathmoveto{\pgfqpoint{0cm}{0cm}}
\pgfpathlineto{\pgfqpoint{1.376cm}{0cm}}
\pgfpathlineto{\pgfqpoint{1.376cm}{1.588cm}}
\pgfpathlineto{\pgfqpoint{0cm}{1.588cm}}
\pgfpathclose
\pgfusepath{clip}
\begin{pgfscope}
\begin{pgfscope}
\pgfpathmoveto{\pgfqpoint{0cm}{0cm}}
\pgfpathlineto{\pgfqpoint{1.376cm}{0cm}}
\pgfpathlineto{\pgfqpoint{1.376cm}{1.588cm}}
\pgfpathlineto{\pgfqpoint{0cm}{1.588cm}}
\pgfpathclose
\pgfusepath{clip}
\begin{pgfscope}
\begin{pgfscope}
\definecolor{eps2pgf_color}{gray}{0.976471}\pgfsetstrokecolor{eps2pgf_color}\pgfsetfillcolor{eps2pgf_color}
\pgfpathmoveto{\pgfqpoint{0cm}{0cm}}
\pgfpathlineto{\pgfqpoint{1.376cm}{0cm}}
\pgfpathlineto{\pgfqpoint{1.376cm}{1.588cm}}
\pgfpathlineto{\pgfqpoint{0cm}{1.588cm}}
\pgfpathclose
\pgfusepath{fill}
\end{pgfscope}
\begin{pgfscope}
\pgfsetdash{}{0cm}
\pgfsetlinewidth{0.818mm}
\pgfsetroundcap
\pgfsetroundjoin
\pgfsetmiterlimit{7.0}
\definecolor{eps2pgf_color}{gray}{0}\pgfsetstrokecolor{eps2pgf_color}\pgfsetfillcolor{eps2pgf_color}
\pgfpathmoveto{\pgfqpoint{0.117cm}{1.476cm}}
\pgfpathlineto{\pgfqpoint{0.682cm}{0.726cm}}
\pgfpathlineto{\pgfqpoint{1.246cm}{1.476cm}}
\pgfusepath{stroke}
\end{pgfscope}
\definecolor{eps2pgf_color}{gray}{0}\pgfsetstrokecolor{eps2pgf_color}\pgfsetfillcolor{eps2pgf_color}
\pgfpathmoveto{\pgfqpoint{0.273cm}{1.451cm}}
\pgfpathcurveto{\pgfqpoint{0.273cm}{1.487cm}}{\pgfqpoint{0.259cm}{1.522cm}}{\pgfqpoint{0.233cm}{1.547cm}}
\pgfpathcurveto{\pgfqpoint{0.207cm}{1.573cm}}{\pgfqpoint{0.173cm}{1.588cm}}{\pgfqpoint{0.137cm}{1.588cm}}
\pgfpathcurveto{\pgfqpoint{0.1cm}{1.588cm}}{\pgfqpoint{0.066cm}{1.573cm}}{\pgfqpoint{0.04cm}{1.547cm}}
\pgfpathcurveto{\pgfqpoint{0.014cm}{1.522cm}}{\pgfqpoint{0cm}{1.487cm}}{\pgfqpoint{0cm}{1.451cm}}
\pgfpathcurveto{\pgfqpoint{0cm}{1.414cm}}{\pgfqpoint{0.014cm}{1.379cm}}{\pgfqpoint{0.04cm}{1.354cm}}
\pgfpathcurveto{\pgfqpoint{0.066cm}{1.328cm}}{\pgfqpoint{0.1cm}{1.314cm}}{\pgfqpoint{0.137cm}{1.314cm}}
\pgfpathcurveto{\pgfqpoint{0.173cm}{1.314cm}}{\pgfqpoint{0.207cm}{1.328cm}}{\pgfqpoint{0.233cm}{1.354cm}}
\pgfpathcurveto{\pgfqpoint{0.259cm}{1.379cm}}{\pgfqpoint{0.273cm}{1.414cm}}{\pgfqpoint{0.273cm}{1.451cm}}
\pgfusepath{fill}
\pgfpathmoveto{\pgfqpoint{1.345cm}{1.426cm}}
\pgfpathcurveto{\pgfqpoint{1.345cm}{1.463cm}}{\pgfqpoint{1.331cm}{1.497cm}}{\pgfqpoint{1.305cm}{1.523cm}}
\pgfpathcurveto{\pgfqpoint{1.28cm}{1.549cm}}{\pgfqpoint{1.245cm}{1.563cm}}{\pgfqpoint{1.209cm}{1.563cm}}
\pgfpathcurveto{\pgfqpoint{1.172cm}{1.563cm}}{\pgfqpoint{1.138cm}{1.549cm}}{\pgfqpoint{1.112cm}{1.523cm}}
\pgfpathcurveto{\pgfqpoint{1.087cm}{1.497cm}}{\pgfqpoint{1.072cm}{1.463cm}}{\pgfqpoint{1.072cm}{1.426cm}}
\pgfpathcurveto{\pgfqpoint{1.072cm}{1.39cm}}{\pgfqpoint{1.087cm}{1.355cm}}{\pgfqpoint{1.112cm}{1.329cm}}
\pgfpathcurveto{\pgfqpoint{1.138cm}{1.304cm}}{\pgfqpoint{1.172cm}{1.289cm}}{\pgfqpoint{1.209cm}{1.289cm}}
\pgfpathcurveto{\pgfqpoint{1.245cm}{1.289cm}}{\pgfqpoint{1.28cm}{1.304cm}}{\pgfqpoint{1.305cm}{1.329cm}}
\pgfpathcurveto{\pgfqpoint{1.331cm}{1.355cm}}{\pgfqpoint{1.345cm}{1.39cm}}{\pgfqpoint{1.345cm}{1.426cm}}
\pgfusepath{fill}
\begin{pgfscope}
\pgfsetdash{}{0cm}
\pgfsetlinewidth{0.818mm}
\pgfsetroundcap
\pgfsetmiterlimit{4.0}
\pgfpathmoveto{\pgfqpoint{0.682cm}{0.726cm}}
\pgfpathlineto{\pgfqpoint{0.682cm}{0.097cm}}
\pgfusepath{stroke}
\end{pgfscope}
\end{pgfscope}
\end{pgfscope}
\end{pgfscope}
\end{tikzpicture}}}\\
&\quad+3X(-X^{\!\resizebox{0.6em}{!}{
\begin{tikzpicture}
\pgfpathmoveto{\pgfqpoint{0cm}{-0.035cm}}
\pgfpathlineto{\pgfqpoint{1.376cm}{-0.035cm}}
\pgfpathlineto{\pgfqpoint{1.376cm}{1.552cm}}
\pgfpathlineto{\pgfqpoint{0cm}{1.552cm}}
\pgfpathclose
\pgfusepath{clip}
\begin{pgfscope}
\begin{pgfscope}
\pgfpathmoveto{\pgfqpoint{0cm}{-0.035cm}}
\pgfpathlineto{\pgfqpoint{1.376cm}{-0.035cm}}
\pgfpathlineto{\pgfqpoint{1.376cm}{1.552cm}}
\pgfpathlineto{\pgfqpoint{0cm}{1.552cm}}
\pgfpathclose
\pgfusepath{clip}
\begin{pgfscope}
\begin{pgfscope}
\pgfsetdash{}{0cm}
\pgfsetlinewidth{0.818mm}
\pgfsetroundcap
\pgfsetroundjoin
\pgfsetmiterlimit{7.0}
\definecolor{eps2pgf_color}{gray}{0}\pgfsetstrokecolor{eps2pgf_color}\pgfsetfillcolor{eps2pgf_color}
\pgfpathmoveto{\pgfqpoint{0.117cm}{1.421cm}}
\pgfpathlineto{\pgfqpoint{0.682cm}{0.671cm}}
\pgfpathlineto{\pgfqpoint{1.246cm}{1.421cm}}
\pgfusepath{stroke}
\end{pgfscope}
\definecolor{eps2pgf_color}{gray}{0}\pgfsetstrokecolor{eps2pgf_color}\pgfsetfillcolor{eps2pgf_color}
\pgfpathmoveto{\pgfqpoint{0.273cm}{1.395cm}}
\pgfpathcurveto{\pgfqpoint{0.273cm}{1.432cm}}{\pgfqpoint{0.259cm}{1.467cm}}{\pgfqpoint{0.233cm}{1.492cm}}
\pgfpathcurveto{\pgfqpoint{0.207cm}{1.518cm}}{\pgfqpoint{0.173cm}{1.532cm}}{\pgfqpoint{0.137cm}{1.532cm}}
\pgfpathcurveto{\pgfqpoint{0.1cm}{1.532cm}}{\pgfqpoint{0.066cm}{1.518cm}}{\pgfqpoint{0.04cm}{1.492cm}}
\pgfpathcurveto{\pgfqpoint{0.014cm}{1.467cm}}{\pgfqpoint{0cm}{1.432cm}}{\pgfqpoint{0cm}{1.395cm}}
\pgfpathcurveto{\pgfqpoint{0cm}{1.359cm}}{\pgfqpoint{0.014cm}{1.324cm}}{\pgfqpoint{0.04cm}{1.299cm}}
\pgfpathcurveto{\pgfqpoint{0.066cm}{1.273cm}}{\pgfqpoint{0.1cm}{1.258cm}}{\pgfqpoint{0.137cm}{1.258cm}}
\pgfpathcurveto{\pgfqpoint{0.173cm}{1.258cm}}{\pgfqpoint{0.207cm}{1.273cm}}{\pgfqpoint{0.233cm}{1.299cm}}
\pgfpathcurveto{\pgfqpoint{0.259cm}{1.324cm}}{\pgfqpoint{0.273cm}{1.359cm}}{\pgfqpoint{0.273cm}{1.395cm}}
\pgfusepath{fill}
\begin{pgfscope}
\pgfsetdash{}{0cm}
\pgfsetlinewidth{0.818mm}
\pgfsetmiterlimit{7.0}
\pgfpathmoveto{\pgfqpoint{0.682cm}{0.671cm}}
\pgfpathlineto{\pgfqpoint{0.679cm}{1.418cm}}
\pgfusepath{stroke}
\end{pgfscope}
\pgfpathmoveto{\pgfqpoint{0.815cm}{1.399cm}}
\pgfpathcurveto{\pgfqpoint{0.815cm}{1.435cm}}{\pgfqpoint{0.801cm}{1.47cm}}{\pgfqpoint{0.775cm}{1.496cm}}
\pgfpathcurveto{\pgfqpoint{0.75cm}{1.521cm}}{\pgfqpoint{0.715cm}{1.536cm}}{\pgfqpoint{0.679cm}{1.536cm}}
\pgfpathcurveto{\pgfqpoint{0.643cm}{1.536cm}}{\pgfqpoint{0.608cm}{1.521cm}}{\pgfqpoint{0.582cm}{1.496cm}}
\pgfpathcurveto{\pgfqpoint{0.557cm}{1.47cm}}{\pgfqpoint{0.542cm}{1.435cm}}{\pgfqpoint{0.542cm}{1.399cm}}
\pgfpathcurveto{\pgfqpoint{0.542cm}{1.363cm}}{\pgfqpoint{0.557cm}{1.328cm}}{\pgfqpoint{0.582cm}{1.302cm}}
\pgfpathcurveto{\pgfqpoint{0.608cm}{1.276cm}}{\pgfqpoint{0.643cm}{1.262cm}}{\pgfqpoint{0.679cm}{1.262cm}}
\pgfpathcurveto{\pgfqpoint{0.715cm}{1.262cm}}{\pgfqpoint{0.75cm}{1.276cm}}{\pgfqpoint{0.775cm}{1.302cm}}
\pgfpathcurveto{\pgfqpoint{0.801cm}{1.328cm}}{\pgfqpoint{0.815cm}{1.363cm}}{\pgfqpoint{0.815cm}{1.399cm}}
\pgfusepath{fill}
\pgfpathmoveto{\pgfqpoint{1.345cm}{1.371cm}}
\pgfpathcurveto{\pgfqpoint{1.345cm}{1.408cm}}{\pgfqpoint{1.331cm}{1.442cm}}{\pgfqpoint{1.305cm}{1.468cm}}
\pgfpathcurveto{\pgfqpoint{1.28cm}{1.494cm}}{\pgfqpoint{1.245cm}{1.508cm}}{\pgfqpoint{1.209cm}{1.508cm}}
\pgfpathcurveto{\pgfqpoint{1.172cm}{1.508cm}}{\pgfqpoint{1.138cm}{1.494cm}}{\pgfqpoint{1.112cm}{1.468cm}}
\pgfpathcurveto{\pgfqpoint{1.087cm}{1.442cm}}{\pgfqpoint{1.072cm}{1.408cm}}{\pgfqpoint{1.072cm}{1.371cm}}
\pgfpathcurveto{\pgfqpoint{1.072cm}{1.335cm}}{\pgfqpoint{1.087cm}{1.3cm}}{\pgfqpoint{1.112cm}{1.274cm}}
\pgfpathcurveto{\pgfqpoint{1.138cm}{1.249cm}}{\pgfqpoint{1.172cm}{1.234cm}}{\pgfqpoint{1.209cm}{1.234cm}}
\pgfpathcurveto{\pgfqpoint{1.245cm}{1.234cm}}{\pgfqpoint{1.28cm}{1.249cm}}{\pgfqpoint{1.305cm}{1.274cm}}
\pgfpathcurveto{\pgfqpoint{1.331cm}{1.3cm}}{\pgfqpoint{1.345cm}{1.335cm}}{\pgfqpoint{1.345cm}{1.371cm}}
\pgfusepath{fill}
\begin{pgfscope}
\pgfsetdash{}{0cm}
\pgfsetlinewidth{0.818mm}
\pgfsetroundcap
\pgfsetmiterlimit{4.0}
\pgfpathmoveto{\pgfqpoint{0.682cm}{0.671cm}}
\pgfpathlineto{\pgfqpoint{0.682cm}{0.042cm}}
\pgfusepath{stroke}
\end{pgfscope}
\end{pgfscope}
\end{pgfscope}
\end{pgfscope}
\end{tikzpicture}}}+\phi+\psi)^2+(-X^{\!\resizebox{0.6em}{!}{
\begin{tikzpicture}
\pgfpathmoveto{\pgfqpoint{0cm}{-0.035cm}}
\pgfpathlineto{\pgfqpoint{1.376cm}{-0.035cm}}
\pgfpathlineto{\pgfqpoint{1.376cm}{1.552cm}}
\pgfpathlineto{\pgfqpoint{0cm}{1.552cm}}
\pgfpathclose
\pgfusepath{clip}
\begin{pgfscope}
\begin{pgfscope}
\pgfpathmoveto{\pgfqpoint{0cm}{-0.035cm}}
\pgfpathlineto{\pgfqpoint{1.376cm}{-0.035cm}}
\pgfpathlineto{\pgfqpoint{1.376cm}{1.552cm}}
\pgfpathlineto{\pgfqpoint{0cm}{1.552cm}}
\pgfpathclose
\pgfusepath{clip}
\begin{pgfscope}
\begin{pgfscope}
\pgfsetdash{}{0cm}
\pgfsetlinewidth{0.818mm}
\pgfsetroundcap
\pgfsetroundjoin
\pgfsetmiterlimit{7.0}
\definecolor{eps2pgf_color}{gray}{0}\pgfsetstrokecolor{eps2pgf_color}\pgfsetfillcolor{eps2pgf_color}
\pgfpathmoveto{\pgfqpoint{0.117cm}{1.421cm}}
\pgfpathlineto{\pgfqpoint{0.682cm}{0.671cm}}
\pgfpathlineto{\pgfqpoint{1.246cm}{1.421cm}}
\pgfusepath{stroke}
\end{pgfscope}
\definecolor{eps2pgf_color}{gray}{0}\pgfsetstrokecolor{eps2pgf_color}\pgfsetfillcolor{eps2pgf_color}
\pgfpathmoveto{\pgfqpoint{0.273cm}{1.395cm}}
\pgfpathcurveto{\pgfqpoint{0.273cm}{1.432cm}}{\pgfqpoint{0.259cm}{1.467cm}}{\pgfqpoint{0.233cm}{1.492cm}}
\pgfpathcurveto{\pgfqpoint{0.207cm}{1.518cm}}{\pgfqpoint{0.173cm}{1.532cm}}{\pgfqpoint{0.137cm}{1.532cm}}
\pgfpathcurveto{\pgfqpoint{0.1cm}{1.532cm}}{\pgfqpoint{0.066cm}{1.518cm}}{\pgfqpoint{0.04cm}{1.492cm}}
\pgfpathcurveto{\pgfqpoint{0.014cm}{1.467cm}}{\pgfqpoint{0cm}{1.432cm}}{\pgfqpoint{0cm}{1.395cm}}
\pgfpathcurveto{\pgfqpoint{0cm}{1.359cm}}{\pgfqpoint{0.014cm}{1.324cm}}{\pgfqpoint{0.04cm}{1.299cm}}
\pgfpathcurveto{\pgfqpoint{0.066cm}{1.273cm}}{\pgfqpoint{0.1cm}{1.258cm}}{\pgfqpoint{0.137cm}{1.258cm}}
\pgfpathcurveto{\pgfqpoint{0.173cm}{1.258cm}}{\pgfqpoint{0.207cm}{1.273cm}}{\pgfqpoint{0.233cm}{1.299cm}}
\pgfpathcurveto{\pgfqpoint{0.259cm}{1.324cm}}{\pgfqpoint{0.273cm}{1.359cm}}{\pgfqpoint{0.273cm}{1.395cm}}
\pgfusepath{fill}
\begin{pgfscope}
\pgfsetdash{}{0cm}
\pgfsetlinewidth{0.818mm}
\pgfsetmiterlimit{7.0}
\pgfpathmoveto{\pgfqpoint{0.682cm}{0.671cm}}
\pgfpathlineto{\pgfqpoint{0.679cm}{1.418cm}}
\pgfusepath{stroke}
\end{pgfscope}
\pgfpathmoveto{\pgfqpoint{0.815cm}{1.399cm}}
\pgfpathcurveto{\pgfqpoint{0.815cm}{1.435cm}}{\pgfqpoint{0.801cm}{1.47cm}}{\pgfqpoint{0.775cm}{1.496cm}}
\pgfpathcurveto{\pgfqpoint{0.75cm}{1.521cm}}{\pgfqpoint{0.715cm}{1.536cm}}{\pgfqpoint{0.679cm}{1.536cm}}
\pgfpathcurveto{\pgfqpoint{0.643cm}{1.536cm}}{\pgfqpoint{0.608cm}{1.521cm}}{\pgfqpoint{0.582cm}{1.496cm}}
\pgfpathcurveto{\pgfqpoint{0.557cm}{1.47cm}}{\pgfqpoint{0.542cm}{1.435cm}}{\pgfqpoint{0.542cm}{1.399cm}}
\pgfpathcurveto{\pgfqpoint{0.542cm}{1.363cm}}{\pgfqpoint{0.557cm}{1.328cm}}{\pgfqpoint{0.582cm}{1.302cm}}
\pgfpathcurveto{\pgfqpoint{0.608cm}{1.276cm}}{\pgfqpoint{0.643cm}{1.262cm}}{\pgfqpoint{0.679cm}{1.262cm}}
\pgfpathcurveto{\pgfqpoint{0.715cm}{1.262cm}}{\pgfqpoint{0.75cm}{1.276cm}}{\pgfqpoint{0.775cm}{1.302cm}}
\pgfpathcurveto{\pgfqpoint{0.801cm}{1.328cm}}{\pgfqpoint{0.815cm}{1.363cm}}{\pgfqpoint{0.815cm}{1.399cm}}
\pgfusepath{fill}
\pgfpathmoveto{\pgfqpoint{1.345cm}{1.371cm}}
\pgfpathcurveto{\pgfqpoint{1.345cm}{1.408cm}}{\pgfqpoint{1.331cm}{1.442cm}}{\pgfqpoint{1.305cm}{1.468cm}}
\pgfpathcurveto{\pgfqpoint{1.28cm}{1.494cm}}{\pgfqpoint{1.245cm}{1.508cm}}{\pgfqpoint{1.209cm}{1.508cm}}
\pgfpathcurveto{\pgfqpoint{1.172cm}{1.508cm}}{\pgfqpoint{1.138cm}{1.494cm}}{\pgfqpoint{1.112cm}{1.468cm}}
\pgfpathcurveto{\pgfqpoint{1.087cm}{1.442cm}}{\pgfqpoint{1.072cm}{1.408cm}}{\pgfqpoint{1.072cm}{1.371cm}}
\pgfpathcurveto{\pgfqpoint{1.072cm}{1.335cm}}{\pgfqpoint{1.087cm}{1.3cm}}{\pgfqpoint{1.112cm}{1.274cm}}
\pgfpathcurveto{\pgfqpoint{1.138cm}{1.249cm}}{\pgfqpoint{1.172cm}{1.234cm}}{\pgfqpoint{1.209cm}{1.234cm}}
\pgfpathcurveto{\pgfqpoint{1.245cm}{1.234cm}}{\pgfqpoint{1.28cm}{1.249cm}}{\pgfqpoint{1.305cm}{1.274cm}}
\pgfpathcurveto{\pgfqpoint{1.331cm}{1.3cm}}{\pgfqpoint{1.345cm}{1.335cm}}{\pgfqpoint{1.345cm}{1.371cm}}
\pgfusepath{fill}
\begin{pgfscope}
\pgfsetdash{}{0cm}
\pgfsetlinewidth{0.818mm}
\pgfsetroundcap
\pgfsetmiterlimit{4.0}
\pgfpathmoveto{\pgfqpoint{0.682cm}{0.671cm}}
\pgfpathlineto{\pgfqpoint{0.682cm}{0.042cm}}
\pgfusepath{stroke}
\end{pgfscope}
\end{pgfscope}
\end{pgfscope}
\end{pgfscope}
\end{tikzpicture}}}+\phi+\psi)^3+3b\varphi.
\end{aligned}
\end{align}

\subsection{Including the localizers}
\label{s:phi}

\rmbb{In this subsection, we introduce a decomposition of \eqref{eq:rhs45} into two equations. To this end, we adopt the following strategy: As the first step,  we decompose the right hand side of \eqref{eq:rhs45} into four parts: in magenta we collect all the contributions of negative regularity containing only various versions of $X$ (belonging at least to $\rmm{\CC^{-1-\kappa}}$), in orange we collect all the terms of negative regularity depending on $\phi+\psi$ (belonging also to $\rmg{\CC^{-1-\kappa}}$), the blue color denotes all the terms belonging locally to $\rmb{L^{\infty}}$ and we keep the term $\psi^3$ separate.
As the next step, we then further decompose each orange term, namely, into a sum of irregular magenta terms (depending also on $\phi+\psi$) and regular blue terms. This leads to the final decomposition \eqref{eq:two} below.}

Within the first step, we write 
\begin{align}\label{eq:2}
&3\llbracket X^2 \rrbracket\succ(-X^{\!\resizebox{0.6em}{!}{
\begin{tikzpicture}
\pgfpathmoveto{\pgfqpoint{0cm}{-0.035cm}}
\pgfpathlineto{\pgfqpoint{1.376cm}{-0.035cm}}
\pgfpathlineto{\pgfqpoint{1.376cm}{1.552cm}}
\pgfpathlineto{\pgfqpoint{0cm}{1.552cm}}
\pgfpathclose
\pgfusepath{clip}
\begin{pgfscope}
\begin{pgfscope}
\pgfpathmoveto{\pgfqpoint{0cm}{-0.035cm}}
\pgfpathlineto{\pgfqpoint{1.376cm}{-0.035cm}}
\pgfpathlineto{\pgfqpoint{1.376cm}{1.552cm}}
\pgfpathlineto{\pgfqpoint{0cm}{1.552cm}}
\pgfpathclose
\pgfusepath{clip}
\begin{pgfscope}
\begin{pgfscope}
\pgfsetdash{}{0cm}
\pgfsetlinewidth{0.818mm}
\pgfsetroundcap
\pgfsetroundjoin
\pgfsetmiterlimit{7.0}
\definecolor{eps2pgf_color}{gray}{0}\pgfsetstrokecolor{eps2pgf_color}\pgfsetfillcolor{eps2pgf_color}
\pgfpathmoveto{\pgfqpoint{0.117cm}{1.421cm}}
\pgfpathlineto{\pgfqpoint{0.682cm}{0.671cm}}
\pgfpathlineto{\pgfqpoint{1.246cm}{1.421cm}}
\pgfusepath{stroke}
\end{pgfscope}
\definecolor{eps2pgf_color}{gray}{0}\pgfsetstrokecolor{eps2pgf_color}\pgfsetfillcolor{eps2pgf_color}
\pgfpathmoveto{\pgfqpoint{0.273cm}{1.395cm}}
\pgfpathcurveto{\pgfqpoint{0.273cm}{1.432cm}}{\pgfqpoint{0.259cm}{1.467cm}}{\pgfqpoint{0.233cm}{1.492cm}}
\pgfpathcurveto{\pgfqpoint{0.207cm}{1.518cm}}{\pgfqpoint{0.173cm}{1.532cm}}{\pgfqpoint{0.137cm}{1.532cm}}
\pgfpathcurveto{\pgfqpoint{0.1cm}{1.532cm}}{\pgfqpoint{0.066cm}{1.518cm}}{\pgfqpoint{0.04cm}{1.492cm}}
\pgfpathcurveto{\pgfqpoint{0.014cm}{1.467cm}}{\pgfqpoint{0cm}{1.432cm}}{\pgfqpoint{0cm}{1.395cm}}
\pgfpathcurveto{\pgfqpoint{0cm}{1.359cm}}{\pgfqpoint{0.014cm}{1.324cm}}{\pgfqpoint{0.04cm}{1.299cm}}
\pgfpathcurveto{\pgfqpoint{0.066cm}{1.273cm}}{\pgfqpoint{0.1cm}{1.258cm}}{\pgfqpoint{0.137cm}{1.258cm}}
\pgfpathcurveto{\pgfqpoint{0.173cm}{1.258cm}}{\pgfqpoint{0.207cm}{1.273cm}}{\pgfqpoint{0.233cm}{1.299cm}}
\pgfpathcurveto{\pgfqpoint{0.259cm}{1.324cm}}{\pgfqpoint{0.273cm}{1.359cm}}{\pgfqpoint{0.273cm}{1.395cm}}
\pgfusepath{fill}
\begin{pgfscope}
\pgfsetdash{}{0cm}
\pgfsetlinewidth{0.818mm}
\pgfsetmiterlimit{7.0}
\pgfpathmoveto{\pgfqpoint{0.682cm}{0.671cm}}
\pgfpathlineto{\pgfqpoint{0.679cm}{1.418cm}}
\pgfusepath{stroke}
\end{pgfscope}
\pgfpathmoveto{\pgfqpoint{0.815cm}{1.399cm}}
\pgfpathcurveto{\pgfqpoint{0.815cm}{1.435cm}}{\pgfqpoint{0.801cm}{1.47cm}}{\pgfqpoint{0.775cm}{1.496cm}}
\pgfpathcurveto{\pgfqpoint{0.75cm}{1.521cm}}{\pgfqpoint{0.715cm}{1.536cm}}{\pgfqpoint{0.679cm}{1.536cm}}
\pgfpathcurveto{\pgfqpoint{0.643cm}{1.536cm}}{\pgfqpoint{0.608cm}{1.521cm}}{\pgfqpoint{0.582cm}{1.496cm}}
\pgfpathcurveto{\pgfqpoint{0.557cm}{1.47cm}}{\pgfqpoint{0.542cm}{1.435cm}}{\pgfqpoint{0.542cm}{1.399cm}}
\pgfpathcurveto{\pgfqpoint{0.542cm}{1.363cm}}{\pgfqpoint{0.557cm}{1.328cm}}{\pgfqpoint{0.582cm}{1.302cm}}
\pgfpathcurveto{\pgfqpoint{0.608cm}{1.276cm}}{\pgfqpoint{0.643cm}{1.262cm}}{\pgfqpoint{0.679cm}{1.262cm}}
\pgfpathcurveto{\pgfqpoint{0.715cm}{1.262cm}}{\pgfqpoint{0.75cm}{1.276cm}}{\pgfqpoint{0.775cm}{1.302cm}}
\pgfpathcurveto{\pgfqpoint{0.801cm}{1.328cm}}{\pgfqpoint{0.815cm}{1.363cm}}{\pgfqpoint{0.815cm}{1.399cm}}
\pgfusepath{fill}
\pgfpathmoveto{\pgfqpoint{1.345cm}{1.371cm}}
\pgfpathcurveto{\pgfqpoint{1.345cm}{1.408cm}}{\pgfqpoint{1.331cm}{1.442cm}}{\pgfqpoint{1.305cm}{1.468cm}}
\pgfpathcurveto{\pgfqpoint{1.28cm}{1.494cm}}{\pgfqpoint{1.245cm}{1.508cm}}{\pgfqpoint{1.209cm}{1.508cm}}
\pgfpathcurveto{\pgfqpoint{1.172cm}{1.508cm}}{\pgfqpoint{1.138cm}{1.494cm}}{\pgfqpoint{1.112cm}{1.468cm}}
\pgfpathcurveto{\pgfqpoint{1.087cm}{1.442cm}}{\pgfqpoint{1.072cm}{1.408cm}}{\pgfqpoint{1.072cm}{1.371cm}}
\pgfpathcurveto{\pgfqpoint{1.072cm}{1.335cm}}{\pgfqpoint{1.087cm}{1.3cm}}{\pgfqpoint{1.112cm}{1.274cm}}
\pgfpathcurveto{\pgfqpoint{1.138cm}{1.249cm}}{\pgfqpoint{1.172cm}{1.234cm}}{\pgfqpoint{1.209cm}{1.234cm}}
\pgfpathcurveto{\pgfqpoint{1.245cm}{1.234cm}}{\pgfqpoint{1.28cm}{1.249cm}}{\pgfqpoint{1.305cm}{1.274cm}}
\pgfpathcurveto{\pgfqpoint{1.331cm}{1.3cm}}{\pgfqpoint{1.345cm}{1.335cm}}{\pgfqpoint{1.345cm}{1.371cm}}
\pgfusepath{fill}
\begin{pgfscope}
\pgfsetdash{}{0cm}
\pgfsetlinewidth{0.818mm}
\pgfsetroundcap
\pgfsetmiterlimit{4.0}
\pgfpathmoveto{\pgfqpoint{0.682cm}{0.671cm}}
\pgfpathlineto{\pgfqpoint{0.682cm}{0.042cm}}
\pgfusepath{stroke}
\end{pgfscope}
\end{pgfscope}
\end{pgfscope}
\end{pgfscope}
\end{tikzpicture}}}+\phi+\psi)=-\rmm{3\llbracket X^2 \rrbracket\succX^{\!\resizebox{0.6em}{!}{
\begin{tikzpicture}
\pgfpathmoveto{\pgfqpoint{0cm}{-0.035cm}}
\pgfpathlineto{\pgfqpoint{1.376cm}{-0.035cm}}
\pgfpathlineto{\pgfqpoint{1.376cm}{1.552cm}}
\pgfpathlineto{\pgfqpoint{0cm}{1.552cm}}
\pgfpathclose
\pgfusepath{clip}
\begin{pgfscope}
\begin{pgfscope}
\pgfpathmoveto{\pgfqpoint{0cm}{-0.035cm}}
\pgfpathlineto{\pgfqpoint{1.376cm}{-0.035cm}}
\pgfpathlineto{\pgfqpoint{1.376cm}{1.552cm}}
\pgfpathlineto{\pgfqpoint{0cm}{1.552cm}}
\pgfpathclose
\pgfusepath{clip}
\begin{pgfscope}
\begin{pgfscope}
\pgfsetdash{}{0cm}
\pgfsetlinewidth{0.818mm}
\pgfsetroundcap
\pgfsetroundjoin
\pgfsetmiterlimit{7.0}
\definecolor{eps2pgf_color}{gray}{0}\pgfsetstrokecolor{eps2pgf_color}\pgfsetfillcolor{eps2pgf_color}
\pgfpathmoveto{\pgfqpoint{0.117cm}{1.421cm}}
\pgfpathlineto{\pgfqpoint{0.682cm}{0.671cm}}
\pgfpathlineto{\pgfqpoint{1.246cm}{1.421cm}}
\pgfusepath{stroke}
\end{pgfscope}
\definecolor{eps2pgf_color}{gray}{0}\pgfsetstrokecolor{eps2pgf_color}\pgfsetfillcolor{eps2pgf_color}
\pgfpathmoveto{\pgfqpoint{0.273cm}{1.395cm}}
\pgfpathcurveto{\pgfqpoint{0.273cm}{1.432cm}}{\pgfqpoint{0.259cm}{1.467cm}}{\pgfqpoint{0.233cm}{1.492cm}}
\pgfpathcurveto{\pgfqpoint{0.207cm}{1.518cm}}{\pgfqpoint{0.173cm}{1.532cm}}{\pgfqpoint{0.137cm}{1.532cm}}
\pgfpathcurveto{\pgfqpoint{0.1cm}{1.532cm}}{\pgfqpoint{0.066cm}{1.518cm}}{\pgfqpoint{0.04cm}{1.492cm}}
\pgfpathcurveto{\pgfqpoint{0.014cm}{1.467cm}}{\pgfqpoint{0cm}{1.432cm}}{\pgfqpoint{0cm}{1.395cm}}
\pgfpathcurveto{\pgfqpoint{0cm}{1.359cm}}{\pgfqpoint{0.014cm}{1.324cm}}{\pgfqpoint{0.04cm}{1.299cm}}
\pgfpathcurveto{\pgfqpoint{0.066cm}{1.273cm}}{\pgfqpoint{0.1cm}{1.258cm}}{\pgfqpoint{0.137cm}{1.258cm}}
\pgfpathcurveto{\pgfqpoint{0.173cm}{1.258cm}}{\pgfqpoint{0.207cm}{1.273cm}}{\pgfqpoint{0.233cm}{1.299cm}}
\pgfpathcurveto{\pgfqpoint{0.259cm}{1.324cm}}{\pgfqpoint{0.273cm}{1.359cm}}{\pgfqpoint{0.273cm}{1.395cm}}
\pgfusepath{fill}
\begin{pgfscope}
\pgfsetdash{}{0cm}
\pgfsetlinewidth{0.818mm}
\pgfsetmiterlimit{7.0}
\pgfpathmoveto{\pgfqpoint{0.682cm}{0.671cm}}
\pgfpathlineto{\pgfqpoint{0.679cm}{1.418cm}}
\pgfusepath{stroke}
\end{pgfscope}
\pgfpathmoveto{\pgfqpoint{0.815cm}{1.399cm}}
\pgfpathcurveto{\pgfqpoint{0.815cm}{1.435cm}}{\pgfqpoint{0.801cm}{1.47cm}}{\pgfqpoint{0.775cm}{1.496cm}}
\pgfpathcurveto{\pgfqpoint{0.75cm}{1.521cm}}{\pgfqpoint{0.715cm}{1.536cm}}{\pgfqpoint{0.679cm}{1.536cm}}
\pgfpathcurveto{\pgfqpoint{0.643cm}{1.536cm}}{\pgfqpoint{0.608cm}{1.521cm}}{\pgfqpoint{0.582cm}{1.496cm}}
\pgfpathcurveto{\pgfqpoint{0.557cm}{1.47cm}}{\pgfqpoint{0.542cm}{1.435cm}}{\pgfqpoint{0.542cm}{1.399cm}}
\pgfpathcurveto{\pgfqpoint{0.542cm}{1.363cm}}{\pgfqpoint{0.557cm}{1.328cm}}{\pgfqpoint{0.582cm}{1.302cm}}
\pgfpathcurveto{\pgfqpoint{0.608cm}{1.276cm}}{\pgfqpoint{0.643cm}{1.262cm}}{\pgfqpoint{0.679cm}{1.262cm}}
\pgfpathcurveto{\pgfqpoint{0.715cm}{1.262cm}}{\pgfqpoint{0.75cm}{1.276cm}}{\pgfqpoint{0.775cm}{1.302cm}}
\pgfpathcurveto{\pgfqpoint{0.801cm}{1.328cm}}{\pgfqpoint{0.815cm}{1.363cm}}{\pgfqpoint{0.815cm}{1.399cm}}
\pgfusepath{fill}
\pgfpathmoveto{\pgfqpoint{1.345cm}{1.371cm}}
\pgfpathcurveto{\pgfqpoint{1.345cm}{1.408cm}}{\pgfqpoint{1.331cm}{1.442cm}}{\pgfqpoint{1.305cm}{1.468cm}}
\pgfpathcurveto{\pgfqpoint{1.28cm}{1.494cm}}{\pgfqpoint{1.245cm}{1.508cm}}{\pgfqpoint{1.209cm}{1.508cm}}
\pgfpathcurveto{\pgfqpoint{1.172cm}{1.508cm}}{\pgfqpoint{1.138cm}{1.494cm}}{\pgfqpoint{1.112cm}{1.468cm}}
\pgfpathcurveto{\pgfqpoint{1.087cm}{1.442cm}}{\pgfqpoint{1.072cm}{1.408cm}}{\pgfqpoint{1.072cm}{1.371cm}}
\pgfpathcurveto{\pgfqpoint{1.072cm}{1.335cm}}{\pgfqpoint{1.087cm}{1.3cm}}{\pgfqpoint{1.112cm}{1.274cm}}
\pgfpathcurveto{\pgfqpoint{1.138cm}{1.249cm}}{\pgfqpoint{1.172cm}{1.234cm}}{\pgfqpoint{1.209cm}{1.234cm}}
\pgfpathcurveto{\pgfqpoint{1.245cm}{1.234cm}}{\pgfqpoint{1.28cm}{1.249cm}}{\pgfqpoint{1.305cm}{1.274cm}}
\pgfpathcurveto{\pgfqpoint{1.331cm}{1.3cm}}{\pgfqpoint{1.345cm}{1.335cm}}{\pgfqpoint{1.345cm}{1.371cm}}
\pgfusepath{fill}
\begin{pgfscope}
\pgfsetdash{}{0cm}
\pgfsetlinewidth{0.818mm}
\pgfsetroundcap
\pgfsetmiterlimit{4.0}
\pgfpathmoveto{\pgfqpoint{0.682cm}{0.671cm}}
\pgfpathlineto{\pgfqpoint{0.682cm}{0.042cm}}
\pgfusepath{stroke}
\end{pgfscope}
\end{pgfscope}
\end{pgfscope}
\end{pgfscope}
\end{tikzpicture}}}}+\rmg{3\llbracket X^2 \rrbracket\succ(\phi+\psi)},
\end{align}
\begin{align*}
&3\llbracket X^2 \rrbracket\preccurlyeq(-X^{\!\resizebox{0.6em}{!}{
\begin{tikzpicture}
\pgfpathmoveto{\pgfqpoint{0cm}{-0.035cm}}
\pgfpathlineto{\pgfqpoint{1.376cm}{-0.035cm}}
\pgfpathlineto{\pgfqpoint{1.376cm}{1.552cm}}
\pgfpathlineto{\pgfqpoint{0cm}{1.552cm}}
\pgfpathclose
\pgfusepath{clip}
\begin{pgfscope}
\begin{pgfscope}
\pgfpathmoveto{\pgfqpoint{0cm}{-0.035cm}}
\pgfpathlineto{\pgfqpoint{1.376cm}{-0.035cm}}
\pgfpathlineto{\pgfqpoint{1.376cm}{1.552cm}}
\pgfpathlineto{\pgfqpoint{0cm}{1.552cm}}
\pgfpathclose
\pgfusepath{clip}
\begin{pgfscope}
\begin{pgfscope}
\pgfsetdash{}{0cm}
\pgfsetlinewidth{0.818mm}
\pgfsetroundcap
\pgfsetroundjoin
\pgfsetmiterlimit{7.0}
\definecolor{eps2pgf_color}{gray}{0}\pgfsetstrokecolor{eps2pgf_color}\pgfsetfillcolor{eps2pgf_color}
\pgfpathmoveto{\pgfqpoint{0.117cm}{1.421cm}}
\pgfpathlineto{\pgfqpoint{0.682cm}{0.671cm}}
\pgfpathlineto{\pgfqpoint{1.246cm}{1.421cm}}
\pgfusepath{stroke}
\end{pgfscope}
\definecolor{eps2pgf_color}{gray}{0}\pgfsetstrokecolor{eps2pgf_color}\pgfsetfillcolor{eps2pgf_color}
\pgfpathmoveto{\pgfqpoint{0.273cm}{1.395cm}}
\pgfpathcurveto{\pgfqpoint{0.273cm}{1.432cm}}{\pgfqpoint{0.259cm}{1.467cm}}{\pgfqpoint{0.233cm}{1.492cm}}
\pgfpathcurveto{\pgfqpoint{0.207cm}{1.518cm}}{\pgfqpoint{0.173cm}{1.532cm}}{\pgfqpoint{0.137cm}{1.532cm}}
\pgfpathcurveto{\pgfqpoint{0.1cm}{1.532cm}}{\pgfqpoint{0.066cm}{1.518cm}}{\pgfqpoint{0.04cm}{1.492cm}}
\pgfpathcurveto{\pgfqpoint{0.014cm}{1.467cm}}{\pgfqpoint{0cm}{1.432cm}}{\pgfqpoint{0cm}{1.395cm}}
\pgfpathcurveto{\pgfqpoint{0cm}{1.359cm}}{\pgfqpoint{0.014cm}{1.324cm}}{\pgfqpoint{0.04cm}{1.299cm}}
\pgfpathcurveto{\pgfqpoint{0.066cm}{1.273cm}}{\pgfqpoint{0.1cm}{1.258cm}}{\pgfqpoint{0.137cm}{1.258cm}}
\pgfpathcurveto{\pgfqpoint{0.173cm}{1.258cm}}{\pgfqpoint{0.207cm}{1.273cm}}{\pgfqpoint{0.233cm}{1.299cm}}
\pgfpathcurveto{\pgfqpoint{0.259cm}{1.324cm}}{\pgfqpoint{0.273cm}{1.359cm}}{\pgfqpoint{0.273cm}{1.395cm}}
\pgfusepath{fill}
\begin{pgfscope}
\pgfsetdash{}{0cm}
\pgfsetlinewidth{0.818mm}
\pgfsetmiterlimit{7.0}
\pgfpathmoveto{\pgfqpoint{0.682cm}{0.671cm}}
\pgfpathlineto{\pgfqpoint{0.679cm}{1.418cm}}
\pgfusepath{stroke}
\end{pgfscope}
\pgfpathmoveto{\pgfqpoint{0.815cm}{1.399cm}}
\pgfpathcurveto{\pgfqpoint{0.815cm}{1.435cm}}{\pgfqpoint{0.801cm}{1.47cm}}{\pgfqpoint{0.775cm}{1.496cm}}
\pgfpathcurveto{\pgfqpoint{0.75cm}{1.521cm}}{\pgfqpoint{0.715cm}{1.536cm}}{\pgfqpoint{0.679cm}{1.536cm}}
\pgfpathcurveto{\pgfqpoint{0.643cm}{1.536cm}}{\pgfqpoint{0.608cm}{1.521cm}}{\pgfqpoint{0.582cm}{1.496cm}}
\pgfpathcurveto{\pgfqpoint{0.557cm}{1.47cm}}{\pgfqpoint{0.542cm}{1.435cm}}{\pgfqpoint{0.542cm}{1.399cm}}
\pgfpathcurveto{\pgfqpoint{0.542cm}{1.363cm}}{\pgfqpoint{0.557cm}{1.328cm}}{\pgfqpoint{0.582cm}{1.302cm}}
\pgfpathcurveto{\pgfqpoint{0.608cm}{1.276cm}}{\pgfqpoint{0.643cm}{1.262cm}}{\pgfqpoint{0.679cm}{1.262cm}}
\pgfpathcurveto{\pgfqpoint{0.715cm}{1.262cm}}{\pgfqpoint{0.75cm}{1.276cm}}{\pgfqpoint{0.775cm}{1.302cm}}
\pgfpathcurveto{\pgfqpoint{0.801cm}{1.328cm}}{\pgfqpoint{0.815cm}{1.363cm}}{\pgfqpoint{0.815cm}{1.399cm}}
\pgfusepath{fill}
\pgfpathmoveto{\pgfqpoint{1.345cm}{1.371cm}}
\pgfpathcurveto{\pgfqpoint{1.345cm}{1.408cm}}{\pgfqpoint{1.331cm}{1.442cm}}{\pgfqpoint{1.305cm}{1.468cm}}
\pgfpathcurveto{\pgfqpoint{1.28cm}{1.494cm}}{\pgfqpoint{1.245cm}{1.508cm}}{\pgfqpoint{1.209cm}{1.508cm}}
\pgfpathcurveto{\pgfqpoint{1.172cm}{1.508cm}}{\pgfqpoint{1.138cm}{1.494cm}}{\pgfqpoint{1.112cm}{1.468cm}}
\pgfpathcurveto{\pgfqpoint{1.087cm}{1.442cm}}{\pgfqpoint{1.072cm}{1.408cm}}{\pgfqpoint{1.072cm}{1.371cm}}
\pgfpathcurveto{\pgfqpoint{1.072cm}{1.335cm}}{\pgfqpoint{1.087cm}{1.3cm}}{\pgfqpoint{1.112cm}{1.274cm}}
\pgfpathcurveto{\pgfqpoint{1.138cm}{1.249cm}}{\pgfqpoint{1.172cm}{1.234cm}}{\pgfqpoint{1.209cm}{1.234cm}}
\pgfpathcurveto{\pgfqpoint{1.245cm}{1.234cm}}{\pgfqpoint{1.28cm}{1.249cm}}{\pgfqpoint{1.305cm}{1.274cm}}
\pgfpathcurveto{\pgfqpoint{1.331cm}{1.3cm}}{\pgfqpoint{1.345cm}{1.335cm}}{\pgfqpoint{1.345cm}{1.371cm}}
\pgfusepath{fill}
\begin{pgfscope}
\pgfsetdash{}{0cm}
\pgfsetlinewidth{0.818mm}
\pgfsetroundcap
\pgfsetmiterlimit{4.0}
\pgfpathmoveto{\pgfqpoint{0.682cm}{0.671cm}}
\pgfpathlineto{\pgfqpoint{0.682cm}{0.042cm}}
\pgfusepath{stroke}
\end{pgfscope}
\end{pgfscope}
\end{pgfscope}
\end{pgfscope}
\end{tikzpicture}}}+\phi+\psi)=-\rmm{3\llbracket X^2 \rrbracket\precX^{\!\resizebox{0.6em}{!}{
\begin{tikzpicture}
\pgfpathmoveto{\pgfqpoint{0cm}{-0.035cm}}
\pgfpathlineto{\pgfqpoint{1.376cm}{-0.035cm}}
\pgfpathlineto{\pgfqpoint{1.376cm}{1.552cm}}
\pgfpathlineto{\pgfqpoint{0cm}{1.552cm}}
\pgfpathclose
\pgfusepath{clip}
\begin{pgfscope}
\begin{pgfscope}
\pgfpathmoveto{\pgfqpoint{0cm}{-0.035cm}}
\pgfpathlineto{\pgfqpoint{1.376cm}{-0.035cm}}
\pgfpathlineto{\pgfqpoint{1.376cm}{1.552cm}}
\pgfpathlineto{\pgfqpoint{0cm}{1.552cm}}
\pgfpathclose
\pgfusepath{clip}
\begin{pgfscope}
\begin{pgfscope}
\pgfsetdash{}{0cm}
\pgfsetlinewidth{0.818mm}
\pgfsetroundcap
\pgfsetroundjoin
\pgfsetmiterlimit{7.0}
\definecolor{eps2pgf_color}{gray}{0}\pgfsetstrokecolor{eps2pgf_color}\pgfsetfillcolor{eps2pgf_color}
\pgfpathmoveto{\pgfqpoint{0.117cm}{1.421cm}}
\pgfpathlineto{\pgfqpoint{0.682cm}{0.671cm}}
\pgfpathlineto{\pgfqpoint{1.246cm}{1.421cm}}
\pgfusepath{stroke}
\end{pgfscope}
\definecolor{eps2pgf_color}{gray}{0}\pgfsetstrokecolor{eps2pgf_color}\pgfsetfillcolor{eps2pgf_color}
\pgfpathmoveto{\pgfqpoint{0.273cm}{1.395cm}}
\pgfpathcurveto{\pgfqpoint{0.273cm}{1.432cm}}{\pgfqpoint{0.259cm}{1.467cm}}{\pgfqpoint{0.233cm}{1.492cm}}
\pgfpathcurveto{\pgfqpoint{0.207cm}{1.518cm}}{\pgfqpoint{0.173cm}{1.532cm}}{\pgfqpoint{0.137cm}{1.532cm}}
\pgfpathcurveto{\pgfqpoint{0.1cm}{1.532cm}}{\pgfqpoint{0.066cm}{1.518cm}}{\pgfqpoint{0.04cm}{1.492cm}}
\pgfpathcurveto{\pgfqpoint{0.014cm}{1.467cm}}{\pgfqpoint{0cm}{1.432cm}}{\pgfqpoint{0cm}{1.395cm}}
\pgfpathcurveto{\pgfqpoint{0cm}{1.359cm}}{\pgfqpoint{0.014cm}{1.324cm}}{\pgfqpoint{0.04cm}{1.299cm}}
\pgfpathcurveto{\pgfqpoint{0.066cm}{1.273cm}}{\pgfqpoint{0.1cm}{1.258cm}}{\pgfqpoint{0.137cm}{1.258cm}}
\pgfpathcurveto{\pgfqpoint{0.173cm}{1.258cm}}{\pgfqpoint{0.207cm}{1.273cm}}{\pgfqpoint{0.233cm}{1.299cm}}
\pgfpathcurveto{\pgfqpoint{0.259cm}{1.324cm}}{\pgfqpoint{0.273cm}{1.359cm}}{\pgfqpoint{0.273cm}{1.395cm}}
\pgfusepath{fill}
\begin{pgfscope}
\pgfsetdash{}{0cm}
\pgfsetlinewidth{0.818mm}
\pgfsetmiterlimit{7.0}
\pgfpathmoveto{\pgfqpoint{0.682cm}{0.671cm}}
\pgfpathlineto{\pgfqpoint{0.679cm}{1.418cm}}
\pgfusepath{stroke}
\end{pgfscope}
\pgfpathmoveto{\pgfqpoint{0.815cm}{1.399cm}}
\pgfpathcurveto{\pgfqpoint{0.815cm}{1.435cm}}{\pgfqpoint{0.801cm}{1.47cm}}{\pgfqpoint{0.775cm}{1.496cm}}
\pgfpathcurveto{\pgfqpoint{0.75cm}{1.521cm}}{\pgfqpoint{0.715cm}{1.536cm}}{\pgfqpoint{0.679cm}{1.536cm}}
\pgfpathcurveto{\pgfqpoint{0.643cm}{1.536cm}}{\pgfqpoint{0.608cm}{1.521cm}}{\pgfqpoint{0.582cm}{1.496cm}}
\pgfpathcurveto{\pgfqpoint{0.557cm}{1.47cm}}{\pgfqpoint{0.542cm}{1.435cm}}{\pgfqpoint{0.542cm}{1.399cm}}
\pgfpathcurveto{\pgfqpoint{0.542cm}{1.363cm}}{\pgfqpoint{0.557cm}{1.328cm}}{\pgfqpoint{0.582cm}{1.302cm}}
\pgfpathcurveto{\pgfqpoint{0.608cm}{1.276cm}}{\pgfqpoint{0.643cm}{1.262cm}}{\pgfqpoint{0.679cm}{1.262cm}}
\pgfpathcurveto{\pgfqpoint{0.715cm}{1.262cm}}{\pgfqpoint{0.75cm}{1.276cm}}{\pgfqpoint{0.775cm}{1.302cm}}
\pgfpathcurveto{\pgfqpoint{0.801cm}{1.328cm}}{\pgfqpoint{0.815cm}{1.363cm}}{\pgfqpoint{0.815cm}{1.399cm}}
\pgfusepath{fill}
\pgfpathmoveto{\pgfqpoint{1.345cm}{1.371cm}}
\pgfpathcurveto{\pgfqpoint{1.345cm}{1.408cm}}{\pgfqpoint{1.331cm}{1.442cm}}{\pgfqpoint{1.305cm}{1.468cm}}
\pgfpathcurveto{\pgfqpoint{1.28cm}{1.494cm}}{\pgfqpoint{1.245cm}{1.508cm}}{\pgfqpoint{1.209cm}{1.508cm}}
\pgfpathcurveto{\pgfqpoint{1.172cm}{1.508cm}}{\pgfqpoint{1.138cm}{1.494cm}}{\pgfqpoint{1.112cm}{1.468cm}}
\pgfpathcurveto{\pgfqpoint{1.087cm}{1.442cm}}{\pgfqpoint{1.072cm}{1.408cm}}{\pgfqpoint{1.072cm}{1.371cm}}
\pgfpathcurveto{\pgfqpoint{1.072cm}{1.335cm}}{\pgfqpoint{1.087cm}{1.3cm}}{\pgfqpoint{1.112cm}{1.274cm}}
\pgfpathcurveto{\pgfqpoint{1.138cm}{1.249cm}}{\pgfqpoint{1.172cm}{1.234cm}}{\pgfqpoint{1.209cm}{1.234cm}}
\pgfpathcurveto{\pgfqpoint{1.245cm}{1.234cm}}{\pgfqpoint{1.28cm}{1.249cm}}{\pgfqpoint{1.305cm}{1.274cm}}
\pgfpathcurveto{\pgfqpoint{1.331cm}{1.3cm}}{\pgfqpoint{1.345cm}{1.335cm}}{\pgfqpoint{1.345cm}{1.371cm}}
\pgfusepath{fill}
\begin{pgfscope}
\pgfsetdash{}{0cm}
\pgfsetlinewidth{0.818mm}
\pgfsetroundcap
\pgfsetmiterlimit{4.0}
\pgfpathmoveto{\pgfqpoint{0.682cm}{0.671cm}}
\pgfpathlineto{\pgfqpoint{0.682cm}{0.042cm}}
\pgfusepath{stroke}
\end{pgfscope}
\end{pgfscope}
\end{pgfscope}
\end{pgfscope}
\end{tikzpicture}}}}+\rmg{3\llbracket X^2 \rrbracket\prec(\phi+\psi)}+3\llbracket X^2 \rrbracket\circ(-X^{\!\resizebox{0.6em}{!}{
\begin{tikzpicture}
\pgfpathmoveto{\pgfqpoint{0cm}{-0.035cm}}
\pgfpathlineto{\pgfqpoint{1.376cm}{-0.035cm}}
\pgfpathlineto{\pgfqpoint{1.376cm}{1.552cm}}
\pgfpathlineto{\pgfqpoint{0cm}{1.552cm}}
\pgfpathclose
\pgfusepath{clip}
\begin{pgfscope}
\begin{pgfscope}
\pgfpathmoveto{\pgfqpoint{0cm}{-0.035cm}}
\pgfpathlineto{\pgfqpoint{1.376cm}{-0.035cm}}
\pgfpathlineto{\pgfqpoint{1.376cm}{1.552cm}}
\pgfpathlineto{\pgfqpoint{0cm}{1.552cm}}
\pgfpathclose
\pgfusepath{clip}
\begin{pgfscope}
\begin{pgfscope}
\pgfsetdash{}{0cm}
\pgfsetlinewidth{0.818mm}
\pgfsetroundcap
\pgfsetroundjoin
\pgfsetmiterlimit{7.0}
\definecolor{eps2pgf_color}{gray}{0}\pgfsetstrokecolor{eps2pgf_color}\pgfsetfillcolor{eps2pgf_color}
\pgfpathmoveto{\pgfqpoint{0.117cm}{1.421cm}}
\pgfpathlineto{\pgfqpoint{0.682cm}{0.671cm}}
\pgfpathlineto{\pgfqpoint{1.246cm}{1.421cm}}
\pgfusepath{stroke}
\end{pgfscope}
\definecolor{eps2pgf_color}{gray}{0}\pgfsetstrokecolor{eps2pgf_color}\pgfsetfillcolor{eps2pgf_color}
\pgfpathmoveto{\pgfqpoint{0.273cm}{1.395cm}}
\pgfpathcurveto{\pgfqpoint{0.273cm}{1.432cm}}{\pgfqpoint{0.259cm}{1.467cm}}{\pgfqpoint{0.233cm}{1.492cm}}
\pgfpathcurveto{\pgfqpoint{0.207cm}{1.518cm}}{\pgfqpoint{0.173cm}{1.532cm}}{\pgfqpoint{0.137cm}{1.532cm}}
\pgfpathcurveto{\pgfqpoint{0.1cm}{1.532cm}}{\pgfqpoint{0.066cm}{1.518cm}}{\pgfqpoint{0.04cm}{1.492cm}}
\pgfpathcurveto{\pgfqpoint{0.014cm}{1.467cm}}{\pgfqpoint{0cm}{1.432cm}}{\pgfqpoint{0cm}{1.395cm}}
\pgfpathcurveto{\pgfqpoint{0cm}{1.359cm}}{\pgfqpoint{0.014cm}{1.324cm}}{\pgfqpoint{0.04cm}{1.299cm}}
\pgfpathcurveto{\pgfqpoint{0.066cm}{1.273cm}}{\pgfqpoint{0.1cm}{1.258cm}}{\pgfqpoint{0.137cm}{1.258cm}}
\pgfpathcurveto{\pgfqpoint{0.173cm}{1.258cm}}{\pgfqpoint{0.207cm}{1.273cm}}{\pgfqpoint{0.233cm}{1.299cm}}
\pgfpathcurveto{\pgfqpoint{0.259cm}{1.324cm}}{\pgfqpoint{0.273cm}{1.359cm}}{\pgfqpoint{0.273cm}{1.395cm}}
\pgfusepath{fill}
\begin{pgfscope}
\pgfsetdash{}{0cm}
\pgfsetlinewidth{0.818mm}
\pgfsetmiterlimit{7.0}
\pgfpathmoveto{\pgfqpoint{0.682cm}{0.671cm}}
\pgfpathlineto{\pgfqpoint{0.679cm}{1.418cm}}
\pgfusepath{stroke}
\end{pgfscope}
\pgfpathmoveto{\pgfqpoint{0.815cm}{1.399cm}}
\pgfpathcurveto{\pgfqpoint{0.815cm}{1.435cm}}{\pgfqpoint{0.801cm}{1.47cm}}{\pgfqpoint{0.775cm}{1.496cm}}
\pgfpathcurveto{\pgfqpoint{0.75cm}{1.521cm}}{\pgfqpoint{0.715cm}{1.536cm}}{\pgfqpoint{0.679cm}{1.536cm}}
\pgfpathcurveto{\pgfqpoint{0.643cm}{1.536cm}}{\pgfqpoint{0.608cm}{1.521cm}}{\pgfqpoint{0.582cm}{1.496cm}}
\pgfpathcurveto{\pgfqpoint{0.557cm}{1.47cm}}{\pgfqpoint{0.542cm}{1.435cm}}{\pgfqpoint{0.542cm}{1.399cm}}
\pgfpathcurveto{\pgfqpoint{0.542cm}{1.363cm}}{\pgfqpoint{0.557cm}{1.328cm}}{\pgfqpoint{0.582cm}{1.302cm}}
\pgfpathcurveto{\pgfqpoint{0.608cm}{1.276cm}}{\pgfqpoint{0.643cm}{1.262cm}}{\pgfqpoint{0.679cm}{1.262cm}}
\pgfpathcurveto{\pgfqpoint{0.715cm}{1.262cm}}{\pgfqpoint{0.75cm}{1.276cm}}{\pgfqpoint{0.775cm}{1.302cm}}
\pgfpathcurveto{\pgfqpoint{0.801cm}{1.328cm}}{\pgfqpoint{0.815cm}{1.363cm}}{\pgfqpoint{0.815cm}{1.399cm}}
\pgfusepath{fill}
\pgfpathmoveto{\pgfqpoint{1.345cm}{1.371cm}}
\pgfpathcurveto{\pgfqpoint{1.345cm}{1.408cm}}{\pgfqpoint{1.331cm}{1.442cm}}{\pgfqpoint{1.305cm}{1.468cm}}
\pgfpathcurveto{\pgfqpoint{1.28cm}{1.494cm}}{\pgfqpoint{1.245cm}{1.508cm}}{\pgfqpoint{1.209cm}{1.508cm}}
\pgfpathcurveto{\pgfqpoint{1.172cm}{1.508cm}}{\pgfqpoint{1.138cm}{1.494cm}}{\pgfqpoint{1.112cm}{1.468cm}}
\pgfpathcurveto{\pgfqpoint{1.087cm}{1.442cm}}{\pgfqpoint{1.072cm}{1.408cm}}{\pgfqpoint{1.072cm}{1.371cm}}
\pgfpathcurveto{\pgfqpoint{1.072cm}{1.335cm}}{\pgfqpoint{1.087cm}{1.3cm}}{\pgfqpoint{1.112cm}{1.274cm}}
\pgfpathcurveto{\pgfqpoint{1.138cm}{1.249cm}}{\pgfqpoint{1.172cm}{1.234cm}}{\pgfqpoint{1.209cm}{1.234cm}}
\pgfpathcurveto{\pgfqpoint{1.245cm}{1.234cm}}{\pgfqpoint{1.28cm}{1.249cm}}{\pgfqpoint{1.305cm}{1.274cm}}
\pgfpathcurveto{\pgfqpoint{1.331cm}{1.3cm}}{\pgfqpoint{1.345cm}{1.335cm}}{\pgfqpoint{1.345cm}{1.371cm}}
\pgfusepath{fill}
\begin{pgfscope}
\pgfsetdash{}{0cm}
\pgfsetlinewidth{0.818mm}
\pgfsetroundcap
\pgfsetmiterlimit{4.0}
\pgfpathmoveto{\pgfqpoint{0.682cm}{0.671cm}}
\pgfpathlineto{\pgfqpoint{0.682cm}{0.042cm}}
\pgfusepath{stroke}
\end{pgfscope}
\end{pgfscope}
\end{pgfscope}
\end{pgfscope}
\end{tikzpicture}}}+\phi)+\rmb{3\llbracket X^2 \rrbracket\circ\psi}.
\end{align*}
Now we add the last term from the right hand side of \eqref{eq:rhs45a} to obtain
\begin{align*}
&3\llbracket X^2 \rrbracket\circ(-X^{\!\resizebox{0.6em}{!}{
\begin{tikzpicture}
\pgfpathmoveto{\pgfqpoint{0cm}{-0.035cm}}
\pgfpathlineto{\pgfqpoint{1.376cm}{-0.035cm}}
\pgfpathlineto{\pgfqpoint{1.376cm}{1.552cm}}
\pgfpathlineto{\pgfqpoint{0cm}{1.552cm}}
\pgfpathclose
\pgfusepath{clip}
\begin{pgfscope}
\begin{pgfscope}
\pgfpathmoveto{\pgfqpoint{0cm}{-0.035cm}}
\pgfpathlineto{\pgfqpoint{1.376cm}{-0.035cm}}
\pgfpathlineto{\pgfqpoint{1.376cm}{1.552cm}}
\pgfpathlineto{\pgfqpoint{0cm}{1.552cm}}
\pgfpathclose
\pgfusepath{clip}
\begin{pgfscope}
\begin{pgfscope}
\pgfsetdash{}{0cm}
\pgfsetlinewidth{0.818mm}
\pgfsetroundcap
\pgfsetroundjoin
\pgfsetmiterlimit{7.0}
\definecolor{eps2pgf_color}{gray}{0}\pgfsetstrokecolor{eps2pgf_color}\pgfsetfillcolor{eps2pgf_color}
\pgfpathmoveto{\pgfqpoint{0.117cm}{1.421cm}}
\pgfpathlineto{\pgfqpoint{0.682cm}{0.671cm}}
\pgfpathlineto{\pgfqpoint{1.246cm}{1.421cm}}
\pgfusepath{stroke}
\end{pgfscope}
\definecolor{eps2pgf_color}{gray}{0}\pgfsetstrokecolor{eps2pgf_color}\pgfsetfillcolor{eps2pgf_color}
\pgfpathmoveto{\pgfqpoint{0.273cm}{1.395cm}}
\pgfpathcurveto{\pgfqpoint{0.273cm}{1.432cm}}{\pgfqpoint{0.259cm}{1.467cm}}{\pgfqpoint{0.233cm}{1.492cm}}
\pgfpathcurveto{\pgfqpoint{0.207cm}{1.518cm}}{\pgfqpoint{0.173cm}{1.532cm}}{\pgfqpoint{0.137cm}{1.532cm}}
\pgfpathcurveto{\pgfqpoint{0.1cm}{1.532cm}}{\pgfqpoint{0.066cm}{1.518cm}}{\pgfqpoint{0.04cm}{1.492cm}}
\pgfpathcurveto{\pgfqpoint{0.014cm}{1.467cm}}{\pgfqpoint{0cm}{1.432cm}}{\pgfqpoint{0cm}{1.395cm}}
\pgfpathcurveto{\pgfqpoint{0cm}{1.359cm}}{\pgfqpoint{0.014cm}{1.324cm}}{\pgfqpoint{0.04cm}{1.299cm}}
\pgfpathcurveto{\pgfqpoint{0.066cm}{1.273cm}}{\pgfqpoint{0.1cm}{1.258cm}}{\pgfqpoint{0.137cm}{1.258cm}}
\pgfpathcurveto{\pgfqpoint{0.173cm}{1.258cm}}{\pgfqpoint{0.207cm}{1.273cm}}{\pgfqpoint{0.233cm}{1.299cm}}
\pgfpathcurveto{\pgfqpoint{0.259cm}{1.324cm}}{\pgfqpoint{0.273cm}{1.359cm}}{\pgfqpoint{0.273cm}{1.395cm}}
\pgfusepath{fill}
\begin{pgfscope}
\pgfsetdash{}{0cm}
\pgfsetlinewidth{0.818mm}
\pgfsetmiterlimit{7.0}
\pgfpathmoveto{\pgfqpoint{0.682cm}{0.671cm}}
\pgfpathlineto{\pgfqpoint{0.679cm}{1.418cm}}
\pgfusepath{stroke}
\end{pgfscope}
\pgfpathmoveto{\pgfqpoint{0.815cm}{1.399cm}}
\pgfpathcurveto{\pgfqpoint{0.815cm}{1.435cm}}{\pgfqpoint{0.801cm}{1.47cm}}{\pgfqpoint{0.775cm}{1.496cm}}
\pgfpathcurveto{\pgfqpoint{0.75cm}{1.521cm}}{\pgfqpoint{0.715cm}{1.536cm}}{\pgfqpoint{0.679cm}{1.536cm}}
\pgfpathcurveto{\pgfqpoint{0.643cm}{1.536cm}}{\pgfqpoint{0.608cm}{1.521cm}}{\pgfqpoint{0.582cm}{1.496cm}}
\pgfpathcurveto{\pgfqpoint{0.557cm}{1.47cm}}{\pgfqpoint{0.542cm}{1.435cm}}{\pgfqpoint{0.542cm}{1.399cm}}
\pgfpathcurveto{\pgfqpoint{0.542cm}{1.363cm}}{\pgfqpoint{0.557cm}{1.328cm}}{\pgfqpoint{0.582cm}{1.302cm}}
\pgfpathcurveto{\pgfqpoint{0.608cm}{1.276cm}}{\pgfqpoint{0.643cm}{1.262cm}}{\pgfqpoint{0.679cm}{1.262cm}}
\pgfpathcurveto{\pgfqpoint{0.715cm}{1.262cm}}{\pgfqpoint{0.75cm}{1.276cm}}{\pgfqpoint{0.775cm}{1.302cm}}
\pgfpathcurveto{\pgfqpoint{0.801cm}{1.328cm}}{\pgfqpoint{0.815cm}{1.363cm}}{\pgfqpoint{0.815cm}{1.399cm}}
\pgfusepath{fill}
\pgfpathmoveto{\pgfqpoint{1.345cm}{1.371cm}}
\pgfpathcurveto{\pgfqpoint{1.345cm}{1.408cm}}{\pgfqpoint{1.331cm}{1.442cm}}{\pgfqpoint{1.305cm}{1.468cm}}
\pgfpathcurveto{\pgfqpoint{1.28cm}{1.494cm}}{\pgfqpoint{1.245cm}{1.508cm}}{\pgfqpoint{1.209cm}{1.508cm}}
\pgfpathcurveto{\pgfqpoint{1.172cm}{1.508cm}}{\pgfqpoint{1.138cm}{1.494cm}}{\pgfqpoint{1.112cm}{1.468cm}}
\pgfpathcurveto{\pgfqpoint{1.087cm}{1.442cm}}{\pgfqpoint{1.072cm}{1.408cm}}{\pgfqpoint{1.072cm}{1.371cm}}
\pgfpathcurveto{\pgfqpoint{1.072cm}{1.335cm}}{\pgfqpoint{1.087cm}{1.3cm}}{\pgfqpoint{1.112cm}{1.274cm}}
\pgfpathcurveto{\pgfqpoint{1.138cm}{1.249cm}}{\pgfqpoint{1.172cm}{1.234cm}}{\pgfqpoint{1.209cm}{1.234cm}}
\pgfpathcurveto{\pgfqpoint{1.245cm}{1.234cm}}{\pgfqpoint{1.28cm}{1.249cm}}{\pgfqpoint{1.305cm}{1.274cm}}
\pgfpathcurveto{\pgfqpoint{1.331cm}{1.3cm}}{\pgfqpoint{1.345cm}{1.335cm}}{\pgfqpoint{1.345cm}{1.371cm}}
\pgfusepath{fill}
\begin{pgfscope}
\pgfsetdash{}{0cm}
\pgfsetlinewidth{0.818mm}
\pgfsetroundcap
\pgfsetmiterlimit{4.0}
\pgfpathmoveto{\pgfqpoint{0.682cm}{0.671cm}}
\pgfpathlineto{\pgfqpoint{0.682cm}{0.042cm}}
\pgfusepath{stroke}
\end{pgfscope}
\end{pgfscope}
\end{pgfscope}
\end{pgfscope}
\end{tikzpicture}}}+\phi)+3b\varphi=- 3X^{\!\resizebox{!}{.8em}{
\begin{tikzpicture}
\pgfpathmoveto{\pgfqpoint{0cm}{-0.035cm}}
\pgfpathlineto{\pgfqpoint{1.976cm}{-0.035cm}}
\pgfpathlineto{\pgfqpoint{1.976cm}{1.94cm}}
\pgfpathlineto{\pgfqpoint{0cm}{1.94cm}}
\pgfpathclose
\pgfusepath{clip}
\begin{pgfscope}
\begin{pgfscope}
\pgfpathmoveto{\pgfqpoint{0cm}{-0.035cm}}
\pgfpathlineto{\pgfqpoint{1.976cm}{-0.035cm}}
\pgfpathlineto{\pgfqpoint{1.976cm}{1.94cm}}
\pgfpathlineto{\pgfqpoint{0cm}{1.94cm}}
\pgfpathclose
\pgfusepath{clip}
\begin{pgfscope}
\begin{pgfscope}
\pgfsetdash{}{0cm}
\pgfsetlinewidth{0.818mm}
\pgfsetroundcap
\pgfsetroundjoin
\pgfsetmiterlimit{7.0}
\definecolor{eps2pgf_color}{gray}{0}\pgfsetstrokecolor{eps2pgf_color}\pgfsetfillcolor{eps2pgf_color}
\pgfpathmoveto{\pgfqpoint{0.117cm}{1.815cm}}
\pgfpathlineto{\pgfqpoint{0.682cm}{1.065cm}}
\pgfpathlineto{\pgfqpoint{1.246cm}{1.815cm}}
\pgfusepath{stroke}
\end{pgfscope}
\definecolor{eps2pgf_color}{gray}{0}\pgfsetstrokecolor{eps2pgf_color}\pgfsetfillcolor{eps2pgf_color}
\pgfpathmoveto{\pgfqpoint{0.273cm}{1.789cm}}
\pgfpathcurveto{\pgfqpoint{0.273cm}{1.825cm}}{\pgfqpoint{0.259cm}{1.86cm}}{\pgfqpoint{0.233cm}{1.886cm}}
\pgfpathcurveto{\pgfqpoint{0.207cm}{1.912cm}}{\pgfqpoint{0.173cm}{1.926cm}}{\pgfqpoint{0.137cm}{1.926cm}}
\pgfpathcurveto{\pgfqpoint{0.1cm}{1.926cm}}{\pgfqpoint{0.066cm}{1.912cm}}{\pgfqpoint{0.04cm}{1.886cm}}
\pgfpathcurveto{\pgfqpoint{0.014cm}{1.86cm}}{\pgfqpoint{0cm}{1.825cm}}{\pgfqpoint{0cm}{1.789cm}}
\pgfpathcurveto{\pgfqpoint{0cm}{1.753cm}}{\pgfqpoint{0.014cm}{1.718cm}}{\pgfqpoint{0.04cm}{1.692cm}}
\pgfpathcurveto{\pgfqpoint{0.066cm}{1.667cm}}{\pgfqpoint{0.1cm}{1.652cm}}{\pgfqpoint{0.137cm}{1.652cm}}
\pgfpathcurveto{\pgfqpoint{0.173cm}{1.652cm}}{\pgfqpoint{0.207cm}{1.667cm}}{\pgfqpoint{0.233cm}{1.692cm}}
\pgfpathcurveto{\pgfqpoint{0.259cm}{1.718cm}}{\pgfqpoint{0.273cm}{1.753cm}}{\pgfqpoint{0.273cm}{1.789cm}}
\pgfusepath{fill}
\begin{pgfscope}
\pgfsetdash{}{0cm}
\pgfsetlinewidth{0.818mm}
\pgfsetmiterlimit{7.0}
\pgfpathmoveto{\pgfqpoint{0.682cm}{1.065cm}}
\pgfpathlineto{\pgfqpoint{0.679cm}{1.812cm}}
\pgfusepath{stroke}
\end{pgfscope}
\pgfpathmoveto{\pgfqpoint{0.815cm}{1.793cm}}
\pgfpathcurveto{\pgfqpoint{0.815cm}{1.829cm}}{\pgfqpoint{0.801cm}{1.864cm}}{\pgfqpoint{0.775cm}{1.89cm}}
\pgfpathcurveto{\pgfqpoint{0.75cm}{1.915cm}}{\pgfqpoint{0.715cm}{1.93cm}}{\pgfqpoint{0.679cm}{1.93cm}}
\pgfpathcurveto{\pgfqpoint{0.643cm}{1.93cm}}{\pgfqpoint{0.608cm}{1.915cm}}{\pgfqpoint{0.582cm}{1.89cm}}
\pgfpathcurveto{\pgfqpoint{0.557cm}{1.864cm}}{\pgfqpoint{0.542cm}{1.829cm}}{\pgfqpoint{0.542cm}{1.793cm}}
\pgfpathcurveto{\pgfqpoint{0.542cm}{1.756cm}}{\pgfqpoint{0.557cm}{1.722cm}}{\pgfqpoint{0.582cm}{1.696cm}}
\pgfpathcurveto{\pgfqpoint{0.608cm}{1.67cm}}{\pgfqpoint{0.643cm}{1.656cm}}{\pgfqpoint{0.679cm}{1.656cm}}
\pgfpathcurveto{\pgfqpoint{0.715cm}{1.656cm}}{\pgfqpoint{0.75cm}{1.67cm}}{\pgfqpoint{0.775cm}{1.696cm}}
\pgfpathcurveto{\pgfqpoint{0.801cm}{1.722cm}}{\pgfqpoint{0.815cm}{1.756cm}}{\pgfqpoint{0.815cm}{1.793cm}}
\pgfusepath{fill}
\pgfpathmoveto{\pgfqpoint{1.345cm}{1.765cm}}
\pgfpathcurveto{\pgfqpoint{1.345cm}{1.801cm}}{\pgfqpoint{1.331cm}{1.836cm}}{\pgfqpoint{1.305cm}{1.862cm}}
\pgfpathcurveto{\pgfqpoint{1.28cm}{1.887cm}}{\pgfqpoint{1.245cm}{1.902cm}}{\pgfqpoint{1.209cm}{1.902cm}}
\pgfpathcurveto{\pgfqpoint{1.172cm}{1.902cm}}{\pgfqpoint{1.138cm}{1.887cm}}{\pgfqpoint{1.112cm}{1.862cm}}
\pgfpathcurveto{\pgfqpoint{1.087cm}{1.836cm}}{\pgfqpoint{1.072cm}{1.801cm}}{\pgfqpoint{1.072cm}{1.765cm}}
\pgfpathcurveto{\pgfqpoint{1.072cm}{1.728cm}}{\pgfqpoint{1.087cm}{1.694cm}}{\pgfqpoint{1.112cm}{1.668cm}}
\pgfpathcurveto{\pgfqpoint{1.138cm}{1.642cm}}{\pgfqpoint{1.172cm}{1.628cm}}{\pgfqpoint{1.209cm}{1.628cm}}
\pgfpathcurveto{\pgfqpoint{1.245cm}{1.628cm}}{\pgfqpoint{1.28cm}{1.642cm}}{\pgfqpoint{1.305cm}{1.668cm}}
\pgfpathcurveto{\pgfqpoint{1.331cm}{1.694cm}}{\pgfqpoint{1.345cm}{1.728cm}}{\pgfqpoint{1.345cm}{1.765cm}}
\pgfusepath{fill}
\begin{pgfscope}
\pgfsetdash{}{0cm}
\pgfsetlinewidth{0.818mm}
\pgfsetroundcap
\pgfsetroundjoin
\pgfsetmiterlimit{7.0}
\pgfpathmoveto{\pgfqpoint{0.682cm}{1.065cm}}
\pgfpathlineto{\pgfqpoint{1.246cm}{0.315cm}}
\pgfpathlineto{\pgfqpoint{1.811cm}{1.065cm}}
\pgfusepath{stroke}
\end{pgfscope}
\pgfpathmoveto{\pgfqpoint{1.948cm}{1.065cm}}
\pgfpathcurveto{\pgfqpoint{1.948cm}{1.101cm}}{\pgfqpoint{1.933cm}{1.136cm}}{\pgfqpoint{1.907cm}{1.162cm}}
\pgfpathcurveto{\pgfqpoint{1.882cm}{1.187cm}}{\pgfqpoint{1.847cm}{1.202cm}}{\pgfqpoint{1.811cm}{1.202cm}}
\pgfpathcurveto{\pgfqpoint{1.775cm}{1.202cm}}{\pgfqpoint{1.74cm}{1.187cm}}{\pgfqpoint{1.714cm}{1.162cm}}
\pgfpathcurveto{\pgfqpoint{1.689cm}{1.136cm}}{\pgfqpoint{1.674cm}{1.101cm}}{\pgfqpoint{1.674cm}{1.065cm}}
\pgfpathcurveto{\pgfqpoint{1.674cm}{1.029cm}}{\pgfqpoint{1.689cm}{0.994cm}}{\pgfqpoint{1.714cm}{0.968cm}}
\pgfpathcurveto{\pgfqpoint{1.74cm}{0.942cm}}{\pgfqpoint{1.775cm}{0.928cm}}{\pgfqpoint{1.811cm}{0.928cm}}
\pgfpathcurveto{\pgfqpoint{1.847cm}{0.928cm}}{\pgfqpoint{1.882cm}{0.942cm}}{\pgfqpoint{1.907cm}{0.968cm}}
\pgfpathcurveto{\pgfqpoint{1.933cm}{0.994cm}}{\pgfqpoint{1.948cm}{1.029cm}}{\pgfqpoint{1.948cm}{1.065cm}}
\pgfusepath{fill}
\begin{pgfscope}
\pgfsetdash{}{0cm}
\pgfsetlinewidth{0.818mm}
\pgfsetmiterlimit{7.0}
\pgfpathmoveto{\pgfqpoint{1.246cm}{0.315cm}}
\pgfpathlineto{\pgfqpoint{1.244cm}{1.061cm}}
\pgfusepath{stroke}
\end{pgfscope}
\pgfpathmoveto{\pgfqpoint{1.38cm}{1.065cm}}
\pgfpathcurveto{\pgfqpoint{1.38cm}{1.101cm}}{\pgfqpoint{1.366cm}{1.136cm}}{\pgfqpoint{1.34cm}{1.162cm}}
\pgfpathcurveto{\pgfqpoint{1.315cm}{1.187cm}}{\pgfqpoint{1.28cm}{1.202cm}}{\pgfqpoint{1.244cm}{1.202cm}}
\pgfpathcurveto{\pgfqpoint{1.207cm}{1.202cm}}{\pgfqpoint{1.173cm}{1.187cm}}{\pgfqpoint{1.147cm}{1.162cm}}
\pgfpathcurveto{\pgfqpoint{1.121cm}{1.136cm}}{\pgfqpoint{1.107cm}{1.101cm}}{\pgfqpoint{1.107cm}{1.065cm}}
\pgfpathcurveto{\pgfqpoint{1.107cm}{1.029cm}}{\pgfqpoint{1.121cm}{0.994cm}}{\pgfqpoint{1.147cm}{0.968cm}}
\pgfpathcurveto{\pgfqpoint{1.173cm}{0.942cm}}{\pgfqpoint{1.207cm}{0.928cm}}{\pgfqpoint{1.244cm}{0.928cm}}
\pgfpathcurveto{\pgfqpoint{1.28cm}{0.928cm}}{\pgfqpoint{1.315cm}{0.942cm}}{\pgfqpoint{1.34cm}{0.968cm}}
\pgfpathcurveto{\pgfqpoint{1.366cm}{0.994cm}}{\pgfqpoint{1.38cm}{1.029cm}}{\pgfqpoint{1.38cm}{1.065cm}}
\pgfusepath{fill}
\begin{pgfscope}
\pgfsetdash{}{0cm}
\pgfsetlinewidth{0.818mm}
\pgfsetmiterlimit{4.0}
\pgfpathmoveto{\pgfqpoint{1.383cm}{0.178cm}}
\pgfpathcurveto{\pgfqpoint{1.383cm}{0.214cm}}{\pgfqpoint{1.369cm}{0.249cm}}{\pgfqpoint{1.343cm}{0.275cm}}
\pgfpathcurveto{\pgfqpoint{1.317cm}{0.3cm}}{\pgfqpoint{1.283cm}{0.315cm}}{\pgfqpoint{1.246cm}{0.315cm}}
\pgfpathcurveto{\pgfqpoint{1.21cm}{0.315cm}}{\pgfqpoint{1.175cm}{0.3cm}}{\pgfqpoint{1.15cm}{0.275cm}}
\pgfpathcurveto{\pgfqpoint{1.124cm}{0.249cm}}{\pgfqpoint{1.11cm}{0.214cm}}{\pgfqpoint{1.11cm}{0.178cm}}
\pgfpathcurveto{\pgfqpoint{1.11cm}{0.141cm}}{\pgfqpoint{1.124cm}{0.107cm}}{\pgfqpoint{1.15cm}{0.081cm}}
\pgfpathcurveto{\pgfqpoint{1.175cm}{0.055cm}}{\pgfqpoint{1.21cm}{0.041cm}}{\pgfqpoint{1.246cm}{0.041cm}}
\pgfpathcurveto{\pgfqpoint{1.283cm}{0.041cm}}{\pgfqpoint{1.317cm}{0.055cm}}{\pgfqpoint{1.343cm}{0.081cm}}
\pgfpathcurveto{\pgfqpoint{1.369cm}{0.107cm}}{\pgfqpoint{1.383cm}{0.141cm}}{\pgfqpoint{1.383cm}{0.178cm}}
\pgfusepath{stroke}
\end{pgfscope}
\end{pgfscope}
\end{pgfscope}
\end{pgfscope}
\end{tikzpicture}}}+3\llbracket X^2 \rrbracket\circ\phi+3b( - X^{\!\resizebox{0.6em}{!}{
\begin{tikzpicture}
\pgfpathmoveto{\pgfqpoint{0cm}{-0.035cm}}
\pgfpathlineto{\pgfqpoint{1.376cm}{-0.035cm}}
\pgfpathlineto{\pgfqpoint{1.376cm}{1.552cm}}
\pgfpathlineto{\pgfqpoint{0cm}{1.552cm}}
\pgfpathclose
\pgfusepath{clip}
\begin{pgfscope}
\begin{pgfscope}
\pgfpathmoveto{\pgfqpoint{0cm}{-0.035cm}}
\pgfpathlineto{\pgfqpoint{1.376cm}{-0.035cm}}
\pgfpathlineto{\pgfqpoint{1.376cm}{1.552cm}}
\pgfpathlineto{\pgfqpoint{0cm}{1.552cm}}
\pgfpathclose
\pgfusepath{clip}
\begin{pgfscope}
\begin{pgfscope}
\pgfsetdash{}{0cm}
\pgfsetlinewidth{0.818mm}
\pgfsetroundcap
\pgfsetroundjoin
\pgfsetmiterlimit{7.0}
\definecolor{eps2pgf_color}{gray}{0}\pgfsetstrokecolor{eps2pgf_color}\pgfsetfillcolor{eps2pgf_color}
\pgfpathmoveto{\pgfqpoint{0.117cm}{1.421cm}}
\pgfpathlineto{\pgfqpoint{0.682cm}{0.671cm}}
\pgfpathlineto{\pgfqpoint{1.246cm}{1.421cm}}
\pgfusepath{stroke}
\end{pgfscope}
\definecolor{eps2pgf_color}{gray}{0}\pgfsetstrokecolor{eps2pgf_color}\pgfsetfillcolor{eps2pgf_color}
\pgfpathmoveto{\pgfqpoint{0.273cm}{1.395cm}}
\pgfpathcurveto{\pgfqpoint{0.273cm}{1.432cm}}{\pgfqpoint{0.259cm}{1.467cm}}{\pgfqpoint{0.233cm}{1.492cm}}
\pgfpathcurveto{\pgfqpoint{0.207cm}{1.518cm}}{\pgfqpoint{0.173cm}{1.532cm}}{\pgfqpoint{0.137cm}{1.532cm}}
\pgfpathcurveto{\pgfqpoint{0.1cm}{1.532cm}}{\pgfqpoint{0.066cm}{1.518cm}}{\pgfqpoint{0.04cm}{1.492cm}}
\pgfpathcurveto{\pgfqpoint{0.014cm}{1.467cm}}{\pgfqpoint{0cm}{1.432cm}}{\pgfqpoint{0cm}{1.395cm}}
\pgfpathcurveto{\pgfqpoint{0cm}{1.359cm}}{\pgfqpoint{0.014cm}{1.324cm}}{\pgfqpoint{0.04cm}{1.299cm}}
\pgfpathcurveto{\pgfqpoint{0.066cm}{1.273cm}}{\pgfqpoint{0.1cm}{1.258cm}}{\pgfqpoint{0.137cm}{1.258cm}}
\pgfpathcurveto{\pgfqpoint{0.173cm}{1.258cm}}{\pgfqpoint{0.207cm}{1.273cm}}{\pgfqpoint{0.233cm}{1.299cm}}
\pgfpathcurveto{\pgfqpoint{0.259cm}{1.324cm}}{\pgfqpoint{0.273cm}{1.359cm}}{\pgfqpoint{0.273cm}{1.395cm}}
\pgfusepath{fill}
\begin{pgfscope}
\pgfsetdash{}{0cm}
\pgfsetlinewidth{0.818mm}
\pgfsetmiterlimit{7.0}
\pgfpathmoveto{\pgfqpoint{0.682cm}{0.671cm}}
\pgfpathlineto{\pgfqpoint{0.679cm}{1.418cm}}
\pgfusepath{stroke}
\end{pgfscope}
\pgfpathmoveto{\pgfqpoint{0.815cm}{1.399cm}}
\pgfpathcurveto{\pgfqpoint{0.815cm}{1.435cm}}{\pgfqpoint{0.801cm}{1.47cm}}{\pgfqpoint{0.775cm}{1.496cm}}
\pgfpathcurveto{\pgfqpoint{0.75cm}{1.521cm}}{\pgfqpoint{0.715cm}{1.536cm}}{\pgfqpoint{0.679cm}{1.536cm}}
\pgfpathcurveto{\pgfqpoint{0.643cm}{1.536cm}}{\pgfqpoint{0.608cm}{1.521cm}}{\pgfqpoint{0.582cm}{1.496cm}}
\pgfpathcurveto{\pgfqpoint{0.557cm}{1.47cm}}{\pgfqpoint{0.542cm}{1.435cm}}{\pgfqpoint{0.542cm}{1.399cm}}
\pgfpathcurveto{\pgfqpoint{0.542cm}{1.363cm}}{\pgfqpoint{0.557cm}{1.328cm}}{\pgfqpoint{0.582cm}{1.302cm}}
\pgfpathcurveto{\pgfqpoint{0.608cm}{1.276cm}}{\pgfqpoint{0.643cm}{1.262cm}}{\pgfqpoint{0.679cm}{1.262cm}}
\pgfpathcurveto{\pgfqpoint{0.715cm}{1.262cm}}{\pgfqpoint{0.75cm}{1.276cm}}{\pgfqpoint{0.775cm}{1.302cm}}
\pgfpathcurveto{\pgfqpoint{0.801cm}{1.328cm}}{\pgfqpoint{0.815cm}{1.363cm}}{\pgfqpoint{0.815cm}{1.399cm}}
\pgfusepath{fill}
\pgfpathmoveto{\pgfqpoint{1.345cm}{1.371cm}}
\pgfpathcurveto{\pgfqpoint{1.345cm}{1.408cm}}{\pgfqpoint{1.331cm}{1.442cm}}{\pgfqpoint{1.305cm}{1.468cm}}
\pgfpathcurveto{\pgfqpoint{1.28cm}{1.494cm}}{\pgfqpoint{1.245cm}{1.508cm}}{\pgfqpoint{1.209cm}{1.508cm}}
\pgfpathcurveto{\pgfqpoint{1.172cm}{1.508cm}}{\pgfqpoint{1.138cm}{1.494cm}}{\pgfqpoint{1.112cm}{1.468cm}}
\pgfpathcurveto{\pgfqpoint{1.087cm}{1.442cm}}{\pgfqpoint{1.072cm}{1.408cm}}{\pgfqpoint{1.072cm}{1.371cm}}
\pgfpathcurveto{\pgfqpoint{1.072cm}{1.335cm}}{\pgfqpoint{1.087cm}{1.3cm}}{\pgfqpoint{1.112cm}{1.274cm}}
\pgfpathcurveto{\pgfqpoint{1.138cm}{1.249cm}}{\pgfqpoint{1.172cm}{1.234cm}}{\pgfqpoint{1.209cm}{1.234cm}}
\pgfpathcurveto{\pgfqpoint{1.245cm}{1.234cm}}{\pgfqpoint{1.28cm}{1.249cm}}{\pgfqpoint{1.305cm}{1.274cm}}
\pgfpathcurveto{\pgfqpoint{1.331cm}{1.3cm}}{\pgfqpoint{1.345cm}{1.335cm}}{\pgfqpoint{1.345cm}{1.371cm}}
\pgfusepath{fill}
\begin{pgfscope}
\pgfsetdash{}{0cm}
\pgfsetlinewidth{0.818mm}
\pgfsetroundcap
\pgfsetmiterlimit{4.0}
\pgfpathmoveto{\pgfqpoint{0.682cm}{0.671cm}}
\pgfpathlineto{\pgfqpoint{0.682cm}{0.042cm}}
\pgfusepath{stroke}
\end{pgfscope}
\end{pgfscope}
\end{pgfscope}
\end{pgfscope}
\end{tikzpicture}}} + \phi + \psi)\\
&\quad= - 3X^{\!\resizebox{!}{.8em}{
\begin{tikzpicture}
\pgfpathmoveto{\pgfqpoint{0cm}{-0.035cm}}
\pgfpathlineto{\pgfqpoint{1.976cm}{-0.035cm}}
\pgfpathlineto{\pgfqpoint{1.976cm}{1.94cm}}
\pgfpathlineto{\pgfqpoint{0cm}{1.94cm}}
\pgfpathclose
\pgfusepath{clip}
\begin{pgfscope}
\begin{pgfscope}
\pgfpathmoveto{\pgfqpoint{0cm}{-0.035cm}}
\pgfpathlineto{\pgfqpoint{1.976cm}{-0.035cm}}
\pgfpathlineto{\pgfqpoint{1.976cm}{1.94cm}}
\pgfpathlineto{\pgfqpoint{0cm}{1.94cm}}
\pgfpathclose
\pgfusepath{clip}
\begin{pgfscope}
\begin{pgfscope}
\pgfsetdash{}{0cm}
\pgfsetlinewidth{0.818mm}
\pgfsetroundcap
\pgfsetroundjoin
\pgfsetmiterlimit{7.0}
\definecolor{eps2pgf_color}{gray}{0}\pgfsetstrokecolor{eps2pgf_color}\pgfsetfillcolor{eps2pgf_color}
\pgfpathmoveto{\pgfqpoint{0.117cm}{1.815cm}}
\pgfpathlineto{\pgfqpoint{0.682cm}{1.065cm}}
\pgfpathlineto{\pgfqpoint{1.246cm}{1.815cm}}
\pgfusepath{stroke}
\end{pgfscope}
\definecolor{eps2pgf_color}{gray}{0}\pgfsetstrokecolor{eps2pgf_color}\pgfsetfillcolor{eps2pgf_color}
\pgfpathmoveto{\pgfqpoint{0.273cm}{1.789cm}}
\pgfpathcurveto{\pgfqpoint{0.273cm}{1.825cm}}{\pgfqpoint{0.259cm}{1.86cm}}{\pgfqpoint{0.233cm}{1.886cm}}
\pgfpathcurveto{\pgfqpoint{0.207cm}{1.912cm}}{\pgfqpoint{0.173cm}{1.926cm}}{\pgfqpoint{0.137cm}{1.926cm}}
\pgfpathcurveto{\pgfqpoint{0.1cm}{1.926cm}}{\pgfqpoint{0.066cm}{1.912cm}}{\pgfqpoint{0.04cm}{1.886cm}}
\pgfpathcurveto{\pgfqpoint{0.014cm}{1.86cm}}{\pgfqpoint{0cm}{1.825cm}}{\pgfqpoint{0cm}{1.789cm}}
\pgfpathcurveto{\pgfqpoint{0cm}{1.753cm}}{\pgfqpoint{0.014cm}{1.718cm}}{\pgfqpoint{0.04cm}{1.692cm}}
\pgfpathcurveto{\pgfqpoint{0.066cm}{1.667cm}}{\pgfqpoint{0.1cm}{1.652cm}}{\pgfqpoint{0.137cm}{1.652cm}}
\pgfpathcurveto{\pgfqpoint{0.173cm}{1.652cm}}{\pgfqpoint{0.207cm}{1.667cm}}{\pgfqpoint{0.233cm}{1.692cm}}
\pgfpathcurveto{\pgfqpoint{0.259cm}{1.718cm}}{\pgfqpoint{0.273cm}{1.753cm}}{\pgfqpoint{0.273cm}{1.789cm}}
\pgfusepath{fill}
\begin{pgfscope}
\pgfsetdash{}{0cm}
\pgfsetlinewidth{0.818mm}
\pgfsetmiterlimit{7.0}
\pgfpathmoveto{\pgfqpoint{0.682cm}{1.065cm}}
\pgfpathlineto{\pgfqpoint{0.679cm}{1.812cm}}
\pgfusepath{stroke}
\end{pgfscope}
\pgfpathmoveto{\pgfqpoint{0.815cm}{1.793cm}}
\pgfpathcurveto{\pgfqpoint{0.815cm}{1.829cm}}{\pgfqpoint{0.801cm}{1.864cm}}{\pgfqpoint{0.775cm}{1.89cm}}
\pgfpathcurveto{\pgfqpoint{0.75cm}{1.915cm}}{\pgfqpoint{0.715cm}{1.93cm}}{\pgfqpoint{0.679cm}{1.93cm}}
\pgfpathcurveto{\pgfqpoint{0.643cm}{1.93cm}}{\pgfqpoint{0.608cm}{1.915cm}}{\pgfqpoint{0.582cm}{1.89cm}}
\pgfpathcurveto{\pgfqpoint{0.557cm}{1.864cm}}{\pgfqpoint{0.542cm}{1.829cm}}{\pgfqpoint{0.542cm}{1.793cm}}
\pgfpathcurveto{\pgfqpoint{0.542cm}{1.756cm}}{\pgfqpoint{0.557cm}{1.722cm}}{\pgfqpoint{0.582cm}{1.696cm}}
\pgfpathcurveto{\pgfqpoint{0.608cm}{1.67cm}}{\pgfqpoint{0.643cm}{1.656cm}}{\pgfqpoint{0.679cm}{1.656cm}}
\pgfpathcurveto{\pgfqpoint{0.715cm}{1.656cm}}{\pgfqpoint{0.75cm}{1.67cm}}{\pgfqpoint{0.775cm}{1.696cm}}
\pgfpathcurveto{\pgfqpoint{0.801cm}{1.722cm}}{\pgfqpoint{0.815cm}{1.756cm}}{\pgfqpoint{0.815cm}{1.793cm}}
\pgfusepath{fill}
\pgfpathmoveto{\pgfqpoint{1.345cm}{1.765cm}}
\pgfpathcurveto{\pgfqpoint{1.345cm}{1.801cm}}{\pgfqpoint{1.331cm}{1.836cm}}{\pgfqpoint{1.305cm}{1.862cm}}
\pgfpathcurveto{\pgfqpoint{1.28cm}{1.887cm}}{\pgfqpoint{1.245cm}{1.902cm}}{\pgfqpoint{1.209cm}{1.902cm}}
\pgfpathcurveto{\pgfqpoint{1.172cm}{1.902cm}}{\pgfqpoint{1.138cm}{1.887cm}}{\pgfqpoint{1.112cm}{1.862cm}}
\pgfpathcurveto{\pgfqpoint{1.087cm}{1.836cm}}{\pgfqpoint{1.072cm}{1.801cm}}{\pgfqpoint{1.072cm}{1.765cm}}
\pgfpathcurveto{\pgfqpoint{1.072cm}{1.728cm}}{\pgfqpoint{1.087cm}{1.694cm}}{\pgfqpoint{1.112cm}{1.668cm}}
\pgfpathcurveto{\pgfqpoint{1.138cm}{1.642cm}}{\pgfqpoint{1.172cm}{1.628cm}}{\pgfqpoint{1.209cm}{1.628cm}}
\pgfpathcurveto{\pgfqpoint{1.245cm}{1.628cm}}{\pgfqpoint{1.28cm}{1.642cm}}{\pgfqpoint{1.305cm}{1.668cm}}
\pgfpathcurveto{\pgfqpoint{1.331cm}{1.694cm}}{\pgfqpoint{1.345cm}{1.728cm}}{\pgfqpoint{1.345cm}{1.765cm}}
\pgfusepath{fill}
\begin{pgfscope}
\pgfsetdash{}{0cm}
\pgfsetlinewidth{0.818mm}
\pgfsetroundcap
\pgfsetroundjoin
\pgfsetmiterlimit{7.0}
\pgfpathmoveto{\pgfqpoint{0.682cm}{1.065cm}}
\pgfpathlineto{\pgfqpoint{1.246cm}{0.315cm}}
\pgfpathlineto{\pgfqpoint{1.811cm}{1.065cm}}
\pgfusepath{stroke}
\end{pgfscope}
\pgfpathmoveto{\pgfqpoint{1.948cm}{1.065cm}}
\pgfpathcurveto{\pgfqpoint{1.948cm}{1.101cm}}{\pgfqpoint{1.933cm}{1.136cm}}{\pgfqpoint{1.907cm}{1.162cm}}
\pgfpathcurveto{\pgfqpoint{1.882cm}{1.187cm}}{\pgfqpoint{1.847cm}{1.202cm}}{\pgfqpoint{1.811cm}{1.202cm}}
\pgfpathcurveto{\pgfqpoint{1.775cm}{1.202cm}}{\pgfqpoint{1.74cm}{1.187cm}}{\pgfqpoint{1.714cm}{1.162cm}}
\pgfpathcurveto{\pgfqpoint{1.689cm}{1.136cm}}{\pgfqpoint{1.674cm}{1.101cm}}{\pgfqpoint{1.674cm}{1.065cm}}
\pgfpathcurveto{\pgfqpoint{1.674cm}{1.029cm}}{\pgfqpoint{1.689cm}{0.994cm}}{\pgfqpoint{1.714cm}{0.968cm}}
\pgfpathcurveto{\pgfqpoint{1.74cm}{0.942cm}}{\pgfqpoint{1.775cm}{0.928cm}}{\pgfqpoint{1.811cm}{0.928cm}}
\pgfpathcurveto{\pgfqpoint{1.847cm}{0.928cm}}{\pgfqpoint{1.882cm}{0.942cm}}{\pgfqpoint{1.907cm}{0.968cm}}
\pgfpathcurveto{\pgfqpoint{1.933cm}{0.994cm}}{\pgfqpoint{1.948cm}{1.029cm}}{\pgfqpoint{1.948cm}{1.065cm}}
\pgfusepath{fill}
\begin{pgfscope}
\pgfsetdash{}{0cm}
\pgfsetlinewidth{0.818mm}
\pgfsetmiterlimit{7.0}
\pgfpathmoveto{\pgfqpoint{1.246cm}{0.315cm}}
\pgfpathlineto{\pgfqpoint{1.244cm}{1.061cm}}
\pgfusepath{stroke}
\end{pgfscope}
\pgfpathmoveto{\pgfqpoint{1.38cm}{1.065cm}}
\pgfpathcurveto{\pgfqpoint{1.38cm}{1.101cm}}{\pgfqpoint{1.366cm}{1.136cm}}{\pgfqpoint{1.34cm}{1.162cm}}
\pgfpathcurveto{\pgfqpoint{1.315cm}{1.187cm}}{\pgfqpoint{1.28cm}{1.202cm}}{\pgfqpoint{1.244cm}{1.202cm}}
\pgfpathcurveto{\pgfqpoint{1.207cm}{1.202cm}}{\pgfqpoint{1.173cm}{1.187cm}}{\pgfqpoint{1.147cm}{1.162cm}}
\pgfpathcurveto{\pgfqpoint{1.121cm}{1.136cm}}{\pgfqpoint{1.107cm}{1.101cm}}{\pgfqpoint{1.107cm}{1.065cm}}
\pgfpathcurveto{\pgfqpoint{1.107cm}{1.029cm}}{\pgfqpoint{1.121cm}{0.994cm}}{\pgfqpoint{1.147cm}{0.968cm}}
\pgfpathcurveto{\pgfqpoint{1.173cm}{0.942cm}}{\pgfqpoint{1.207cm}{0.928cm}}{\pgfqpoint{1.244cm}{0.928cm}}
\pgfpathcurveto{\pgfqpoint{1.28cm}{0.928cm}}{\pgfqpoint{1.315cm}{0.942cm}}{\pgfqpoint{1.34cm}{0.968cm}}
\pgfpathcurveto{\pgfqpoint{1.366cm}{0.994cm}}{\pgfqpoint{1.38cm}{1.029cm}}{\pgfqpoint{1.38cm}{1.065cm}}
\pgfusepath{fill}
\begin{pgfscope}
\pgfsetdash{}{0cm}
\pgfsetlinewidth{0.818mm}
\pgfsetmiterlimit{4.0}
\pgfpathmoveto{\pgfqpoint{1.383cm}{0.178cm}}
\pgfpathcurveto{\pgfqpoint{1.383cm}{0.214cm}}{\pgfqpoint{1.369cm}{0.249cm}}{\pgfqpoint{1.343cm}{0.275cm}}
\pgfpathcurveto{\pgfqpoint{1.317cm}{0.3cm}}{\pgfqpoint{1.283cm}{0.315cm}}{\pgfqpoint{1.246cm}{0.315cm}}
\pgfpathcurveto{\pgfqpoint{1.21cm}{0.315cm}}{\pgfqpoint{1.175cm}{0.3cm}}{\pgfqpoint{1.15cm}{0.275cm}}
\pgfpathcurveto{\pgfqpoint{1.124cm}{0.249cm}}{\pgfqpoint{1.11cm}{0.214cm}}{\pgfqpoint{1.11cm}{0.178cm}}
\pgfpathcurveto{\pgfqpoint{1.11cm}{0.141cm}}{\pgfqpoint{1.124cm}{0.107cm}}{\pgfqpoint{1.15cm}{0.081cm}}
\pgfpathcurveto{\pgfqpoint{1.175cm}{0.055cm}}{\pgfqpoint{1.21cm}{0.041cm}}{\pgfqpoint{1.246cm}{0.041cm}}
\pgfpathcurveto{\pgfqpoint{1.283cm}{0.041cm}}{\pgfqpoint{1.317cm}{0.055cm}}{\pgfqpoint{1.343cm}{0.081cm}}
\pgfpathcurveto{\pgfqpoint{1.369cm}{0.107cm}}{\pgfqpoint{1.383cm}{0.141cm}}{\pgfqpoint{1.383cm}{0.178cm}}
\pgfusepath{stroke}
\end{pgfscope}
\end{pgfscope}
\end{pgfscope}
\end{pgfscope}
\end{tikzpicture}}}+3\llbracket X^2 \rrbracket\circ\vartheta-9\llbracket X^2 \rrbracket\circ(( - X^{\!\resizebox{0.6em}{!}{
\begin{tikzpicture}
\pgfpathmoveto{\pgfqpoint{0cm}{-0.035cm}}
\pgfpathlineto{\pgfqpoint{1.376cm}{-0.035cm}}
\pgfpathlineto{\pgfqpoint{1.376cm}{1.552cm}}
\pgfpathlineto{\pgfqpoint{0cm}{1.552cm}}
\pgfpathclose
\pgfusepath{clip}
\begin{pgfscope}
\begin{pgfscope}
\pgfpathmoveto{\pgfqpoint{0cm}{-0.035cm}}
\pgfpathlineto{\pgfqpoint{1.376cm}{-0.035cm}}
\pgfpathlineto{\pgfqpoint{1.376cm}{1.552cm}}
\pgfpathlineto{\pgfqpoint{0cm}{1.552cm}}
\pgfpathclose
\pgfusepath{clip}
\begin{pgfscope}
\begin{pgfscope}
\pgfsetdash{}{0cm}
\pgfsetlinewidth{0.818mm}
\pgfsetroundcap
\pgfsetroundjoin
\pgfsetmiterlimit{7.0}
\definecolor{eps2pgf_color}{gray}{0}\pgfsetstrokecolor{eps2pgf_color}\pgfsetfillcolor{eps2pgf_color}
\pgfpathmoveto{\pgfqpoint{0.117cm}{1.421cm}}
\pgfpathlineto{\pgfqpoint{0.682cm}{0.671cm}}
\pgfpathlineto{\pgfqpoint{1.246cm}{1.421cm}}
\pgfusepath{stroke}
\end{pgfscope}
\definecolor{eps2pgf_color}{gray}{0}\pgfsetstrokecolor{eps2pgf_color}\pgfsetfillcolor{eps2pgf_color}
\pgfpathmoveto{\pgfqpoint{0.273cm}{1.395cm}}
\pgfpathcurveto{\pgfqpoint{0.273cm}{1.432cm}}{\pgfqpoint{0.259cm}{1.467cm}}{\pgfqpoint{0.233cm}{1.492cm}}
\pgfpathcurveto{\pgfqpoint{0.207cm}{1.518cm}}{\pgfqpoint{0.173cm}{1.532cm}}{\pgfqpoint{0.137cm}{1.532cm}}
\pgfpathcurveto{\pgfqpoint{0.1cm}{1.532cm}}{\pgfqpoint{0.066cm}{1.518cm}}{\pgfqpoint{0.04cm}{1.492cm}}
\pgfpathcurveto{\pgfqpoint{0.014cm}{1.467cm}}{\pgfqpoint{0cm}{1.432cm}}{\pgfqpoint{0cm}{1.395cm}}
\pgfpathcurveto{\pgfqpoint{0cm}{1.359cm}}{\pgfqpoint{0.014cm}{1.324cm}}{\pgfqpoint{0.04cm}{1.299cm}}
\pgfpathcurveto{\pgfqpoint{0.066cm}{1.273cm}}{\pgfqpoint{0.1cm}{1.258cm}}{\pgfqpoint{0.137cm}{1.258cm}}
\pgfpathcurveto{\pgfqpoint{0.173cm}{1.258cm}}{\pgfqpoint{0.207cm}{1.273cm}}{\pgfqpoint{0.233cm}{1.299cm}}
\pgfpathcurveto{\pgfqpoint{0.259cm}{1.324cm}}{\pgfqpoint{0.273cm}{1.359cm}}{\pgfqpoint{0.273cm}{1.395cm}}
\pgfusepath{fill}
\begin{pgfscope}
\pgfsetdash{}{0cm}
\pgfsetlinewidth{0.818mm}
\pgfsetmiterlimit{7.0}
\pgfpathmoveto{\pgfqpoint{0.682cm}{0.671cm}}
\pgfpathlineto{\pgfqpoint{0.679cm}{1.418cm}}
\pgfusepath{stroke}
\end{pgfscope}
\pgfpathmoveto{\pgfqpoint{0.815cm}{1.399cm}}
\pgfpathcurveto{\pgfqpoint{0.815cm}{1.435cm}}{\pgfqpoint{0.801cm}{1.47cm}}{\pgfqpoint{0.775cm}{1.496cm}}
\pgfpathcurveto{\pgfqpoint{0.75cm}{1.521cm}}{\pgfqpoint{0.715cm}{1.536cm}}{\pgfqpoint{0.679cm}{1.536cm}}
\pgfpathcurveto{\pgfqpoint{0.643cm}{1.536cm}}{\pgfqpoint{0.608cm}{1.521cm}}{\pgfqpoint{0.582cm}{1.496cm}}
\pgfpathcurveto{\pgfqpoint{0.557cm}{1.47cm}}{\pgfqpoint{0.542cm}{1.435cm}}{\pgfqpoint{0.542cm}{1.399cm}}
\pgfpathcurveto{\pgfqpoint{0.542cm}{1.363cm}}{\pgfqpoint{0.557cm}{1.328cm}}{\pgfqpoint{0.582cm}{1.302cm}}
\pgfpathcurveto{\pgfqpoint{0.608cm}{1.276cm}}{\pgfqpoint{0.643cm}{1.262cm}}{\pgfqpoint{0.679cm}{1.262cm}}
\pgfpathcurveto{\pgfqpoint{0.715cm}{1.262cm}}{\pgfqpoint{0.75cm}{1.276cm}}{\pgfqpoint{0.775cm}{1.302cm}}
\pgfpathcurveto{\pgfqpoint{0.801cm}{1.328cm}}{\pgfqpoint{0.815cm}{1.363cm}}{\pgfqpoint{0.815cm}{1.399cm}}
\pgfusepath{fill}
\pgfpathmoveto{\pgfqpoint{1.345cm}{1.371cm}}
\pgfpathcurveto{\pgfqpoint{1.345cm}{1.408cm}}{\pgfqpoint{1.331cm}{1.442cm}}{\pgfqpoint{1.305cm}{1.468cm}}
\pgfpathcurveto{\pgfqpoint{1.28cm}{1.494cm}}{\pgfqpoint{1.245cm}{1.508cm}}{\pgfqpoint{1.209cm}{1.508cm}}
\pgfpathcurveto{\pgfqpoint{1.172cm}{1.508cm}}{\pgfqpoint{1.138cm}{1.494cm}}{\pgfqpoint{1.112cm}{1.468cm}}
\pgfpathcurveto{\pgfqpoint{1.087cm}{1.442cm}}{\pgfqpoint{1.072cm}{1.408cm}}{\pgfqpoint{1.072cm}{1.371cm}}
\pgfpathcurveto{\pgfqpoint{1.072cm}{1.335cm}}{\pgfqpoint{1.087cm}{1.3cm}}{\pgfqpoint{1.112cm}{1.274cm}}
\pgfpathcurveto{\pgfqpoint{1.138cm}{1.249cm}}{\pgfqpoint{1.172cm}{1.234cm}}{\pgfqpoint{1.209cm}{1.234cm}}
\pgfpathcurveto{\pgfqpoint{1.245cm}{1.234cm}}{\pgfqpoint{1.28cm}{1.249cm}}{\pgfqpoint{1.305cm}{1.274cm}}
\pgfpathcurveto{\pgfqpoint{1.331cm}{1.3cm}}{\pgfqpoint{1.345cm}{1.335cm}}{\pgfqpoint{1.345cm}{1.371cm}}
\pgfusepath{fill}
\begin{pgfscope}
\pgfsetdash{}{0cm}
\pgfsetlinewidth{0.818mm}
\pgfsetroundcap
\pgfsetmiterlimit{4.0}
\pgfpathmoveto{\pgfqpoint{0.682cm}{0.671cm}}
\pgfpathlineto{\pgfqpoint{0.682cm}{0.042cm}}
\pgfusepath{stroke}
\end{pgfscope}
\end{pgfscope}
\end{pgfscope}
\end{pgfscope}
\end{tikzpicture}}} + \phi + \psi)\prec X^{\!\resizebox{0.6em}{!}{
\begin{tikzpicture}
\pgfpathmoveto{\pgfqpoint{0cm}{0cm}}
\pgfpathlineto{\pgfqpoint{1.376cm}{0cm}}
\pgfpathlineto{\pgfqpoint{1.376cm}{1.588cm}}
\pgfpathlineto{\pgfqpoint{0cm}{1.588cm}}
\pgfpathclose
\pgfusepath{clip}
\begin{pgfscope}
\begin{pgfscope}
\pgfpathmoveto{\pgfqpoint{0cm}{0cm}}
\pgfpathlineto{\pgfqpoint{1.376cm}{0cm}}
\pgfpathlineto{\pgfqpoint{1.376cm}{1.588cm}}
\pgfpathlineto{\pgfqpoint{0cm}{1.588cm}}
\pgfpathclose
\pgfusepath{clip}
\begin{pgfscope}
\begin{pgfscope}
\definecolor{eps2pgf_color}{gray}{0.976471}\pgfsetstrokecolor{eps2pgf_color}\pgfsetfillcolor{eps2pgf_color}
\pgfpathmoveto{\pgfqpoint{0cm}{0cm}}
\pgfpathlineto{\pgfqpoint{1.376cm}{0cm}}
\pgfpathlineto{\pgfqpoint{1.376cm}{1.588cm}}
\pgfpathlineto{\pgfqpoint{0cm}{1.588cm}}
\pgfpathclose
\pgfusepath{fill}
\end{pgfscope}
\begin{pgfscope}
\pgfsetdash{}{0cm}
\pgfsetlinewidth{0.818mm}
\pgfsetroundcap
\pgfsetroundjoin
\pgfsetmiterlimit{7.0}
\definecolor{eps2pgf_color}{gray}{0}\pgfsetstrokecolor{eps2pgf_color}\pgfsetfillcolor{eps2pgf_color}
\pgfpathmoveto{\pgfqpoint{0.117cm}{1.476cm}}
\pgfpathlineto{\pgfqpoint{0.682cm}{0.726cm}}
\pgfpathlineto{\pgfqpoint{1.246cm}{1.476cm}}
\pgfusepath{stroke}
\end{pgfscope}
\definecolor{eps2pgf_color}{gray}{0}\pgfsetstrokecolor{eps2pgf_color}\pgfsetfillcolor{eps2pgf_color}
\pgfpathmoveto{\pgfqpoint{0.273cm}{1.451cm}}
\pgfpathcurveto{\pgfqpoint{0.273cm}{1.487cm}}{\pgfqpoint{0.259cm}{1.522cm}}{\pgfqpoint{0.233cm}{1.547cm}}
\pgfpathcurveto{\pgfqpoint{0.207cm}{1.573cm}}{\pgfqpoint{0.173cm}{1.588cm}}{\pgfqpoint{0.137cm}{1.588cm}}
\pgfpathcurveto{\pgfqpoint{0.1cm}{1.588cm}}{\pgfqpoint{0.066cm}{1.573cm}}{\pgfqpoint{0.04cm}{1.547cm}}
\pgfpathcurveto{\pgfqpoint{0.014cm}{1.522cm}}{\pgfqpoint{0cm}{1.487cm}}{\pgfqpoint{0cm}{1.451cm}}
\pgfpathcurveto{\pgfqpoint{0cm}{1.414cm}}{\pgfqpoint{0.014cm}{1.379cm}}{\pgfqpoint{0.04cm}{1.354cm}}
\pgfpathcurveto{\pgfqpoint{0.066cm}{1.328cm}}{\pgfqpoint{0.1cm}{1.314cm}}{\pgfqpoint{0.137cm}{1.314cm}}
\pgfpathcurveto{\pgfqpoint{0.173cm}{1.314cm}}{\pgfqpoint{0.207cm}{1.328cm}}{\pgfqpoint{0.233cm}{1.354cm}}
\pgfpathcurveto{\pgfqpoint{0.259cm}{1.379cm}}{\pgfqpoint{0.273cm}{1.414cm}}{\pgfqpoint{0.273cm}{1.451cm}}
\pgfusepath{fill}
\pgfpathmoveto{\pgfqpoint{1.345cm}{1.426cm}}
\pgfpathcurveto{\pgfqpoint{1.345cm}{1.463cm}}{\pgfqpoint{1.331cm}{1.497cm}}{\pgfqpoint{1.305cm}{1.523cm}}
\pgfpathcurveto{\pgfqpoint{1.28cm}{1.549cm}}{\pgfqpoint{1.245cm}{1.563cm}}{\pgfqpoint{1.209cm}{1.563cm}}
\pgfpathcurveto{\pgfqpoint{1.172cm}{1.563cm}}{\pgfqpoint{1.138cm}{1.549cm}}{\pgfqpoint{1.112cm}{1.523cm}}
\pgfpathcurveto{\pgfqpoint{1.087cm}{1.497cm}}{\pgfqpoint{1.072cm}{1.463cm}}{\pgfqpoint{1.072cm}{1.426cm}}
\pgfpathcurveto{\pgfqpoint{1.072cm}{1.39cm}}{\pgfqpoint{1.087cm}{1.355cm}}{\pgfqpoint{1.112cm}{1.329cm}}
\pgfpathcurveto{\pgfqpoint{1.138cm}{1.304cm}}{\pgfqpoint{1.172cm}{1.289cm}}{\pgfqpoint{1.209cm}{1.289cm}}
\pgfpathcurveto{\pgfqpoint{1.245cm}{1.289cm}}{\pgfqpoint{1.28cm}{1.304cm}}{\pgfqpoint{1.305cm}{1.329cm}}
\pgfpathcurveto{\pgfqpoint{1.331cm}{1.355cm}}{\pgfqpoint{1.345cm}{1.39cm}}{\pgfqpoint{1.345cm}{1.426cm}}
\pgfusepath{fill}
\begin{pgfscope}
\pgfsetdash{}{0cm}
\pgfsetlinewidth{0.818mm}
\pgfsetroundcap
\pgfsetmiterlimit{4.0}
\pgfpathmoveto{\pgfqpoint{0.682cm}{0.726cm}}
\pgfpathlineto{\pgfqpoint{0.682cm}{0.097cm}}
\pgfusepath{stroke}
\end{pgfscope}
\end{pgfscope}
\end{pgfscope}
\end{pgfscope}
\end{tikzpicture}}})\\
&\qquad\qquad+3b( - X^{\!\resizebox{0.6em}{!}{
\begin{tikzpicture}
\pgfpathmoveto{\pgfqpoint{0cm}{-0.035cm}}
\pgfpathlineto{\pgfqpoint{1.376cm}{-0.035cm}}
\pgfpathlineto{\pgfqpoint{1.376cm}{1.552cm}}
\pgfpathlineto{\pgfqpoint{0cm}{1.552cm}}
\pgfpathclose
\pgfusepath{clip}
\begin{pgfscope}
\begin{pgfscope}
\pgfpathmoveto{\pgfqpoint{0cm}{-0.035cm}}
\pgfpathlineto{\pgfqpoint{1.376cm}{-0.035cm}}
\pgfpathlineto{\pgfqpoint{1.376cm}{1.552cm}}
\pgfpathlineto{\pgfqpoint{0cm}{1.552cm}}
\pgfpathclose
\pgfusepath{clip}
\begin{pgfscope}
\begin{pgfscope}
\pgfsetdash{}{0cm}
\pgfsetlinewidth{0.818mm}
\pgfsetroundcap
\pgfsetroundjoin
\pgfsetmiterlimit{7.0}
\definecolor{eps2pgf_color}{gray}{0}\pgfsetstrokecolor{eps2pgf_color}\pgfsetfillcolor{eps2pgf_color}
\pgfpathmoveto{\pgfqpoint{0.117cm}{1.421cm}}
\pgfpathlineto{\pgfqpoint{0.682cm}{0.671cm}}
\pgfpathlineto{\pgfqpoint{1.246cm}{1.421cm}}
\pgfusepath{stroke}
\end{pgfscope}
\definecolor{eps2pgf_color}{gray}{0}\pgfsetstrokecolor{eps2pgf_color}\pgfsetfillcolor{eps2pgf_color}
\pgfpathmoveto{\pgfqpoint{0.273cm}{1.395cm}}
\pgfpathcurveto{\pgfqpoint{0.273cm}{1.432cm}}{\pgfqpoint{0.259cm}{1.467cm}}{\pgfqpoint{0.233cm}{1.492cm}}
\pgfpathcurveto{\pgfqpoint{0.207cm}{1.518cm}}{\pgfqpoint{0.173cm}{1.532cm}}{\pgfqpoint{0.137cm}{1.532cm}}
\pgfpathcurveto{\pgfqpoint{0.1cm}{1.532cm}}{\pgfqpoint{0.066cm}{1.518cm}}{\pgfqpoint{0.04cm}{1.492cm}}
\pgfpathcurveto{\pgfqpoint{0.014cm}{1.467cm}}{\pgfqpoint{0cm}{1.432cm}}{\pgfqpoint{0cm}{1.395cm}}
\pgfpathcurveto{\pgfqpoint{0cm}{1.359cm}}{\pgfqpoint{0.014cm}{1.324cm}}{\pgfqpoint{0.04cm}{1.299cm}}
\pgfpathcurveto{\pgfqpoint{0.066cm}{1.273cm}}{\pgfqpoint{0.1cm}{1.258cm}}{\pgfqpoint{0.137cm}{1.258cm}}
\pgfpathcurveto{\pgfqpoint{0.173cm}{1.258cm}}{\pgfqpoint{0.207cm}{1.273cm}}{\pgfqpoint{0.233cm}{1.299cm}}
\pgfpathcurveto{\pgfqpoint{0.259cm}{1.324cm}}{\pgfqpoint{0.273cm}{1.359cm}}{\pgfqpoint{0.273cm}{1.395cm}}
\pgfusepath{fill}
\begin{pgfscope}
\pgfsetdash{}{0cm}
\pgfsetlinewidth{0.818mm}
\pgfsetmiterlimit{7.0}
\pgfpathmoveto{\pgfqpoint{0.682cm}{0.671cm}}
\pgfpathlineto{\pgfqpoint{0.679cm}{1.418cm}}
\pgfusepath{stroke}
\end{pgfscope}
\pgfpathmoveto{\pgfqpoint{0.815cm}{1.399cm}}
\pgfpathcurveto{\pgfqpoint{0.815cm}{1.435cm}}{\pgfqpoint{0.801cm}{1.47cm}}{\pgfqpoint{0.775cm}{1.496cm}}
\pgfpathcurveto{\pgfqpoint{0.75cm}{1.521cm}}{\pgfqpoint{0.715cm}{1.536cm}}{\pgfqpoint{0.679cm}{1.536cm}}
\pgfpathcurveto{\pgfqpoint{0.643cm}{1.536cm}}{\pgfqpoint{0.608cm}{1.521cm}}{\pgfqpoint{0.582cm}{1.496cm}}
\pgfpathcurveto{\pgfqpoint{0.557cm}{1.47cm}}{\pgfqpoint{0.542cm}{1.435cm}}{\pgfqpoint{0.542cm}{1.399cm}}
\pgfpathcurveto{\pgfqpoint{0.542cm}{1.363cm}}{\pgfqpoint{0.557cm}{1.328cm}}{\pgfqpoint{0.582cm}{1.302cm}}
\pgfpathcurveto{\pgfqpoint{0.608cm}{1.276cm}}{\pgfqpoint{0.643cm}{1.262cm}}{\pgfqpoint{0.679cm}{1.262cm}}
\pgfpathcurveto{\pgfqpoint{0.715cm}{1.262cm}}{\pgfqpoint{0.75cm}{1.276cm}}{\pgfqpoint{0.775cm}{1.302cm}}
\pgfpathcurveto{\pgfqpoint{0.801cm}{1.328cm}}{\pgfqpoint{0.815cm}{1.363cm}}{\pgfqpoint{0.815cm}{1.399cm}}
\pgfusepath{fill}
\pgfpathmoveto{\pgfqpoint{1.345cm}{1.371cm}}
\pgfpathcurveto{\pgfqpoint{1.345cm}{1.408cm}}{\pgfqpoint{1.331cm}{1.442cm}}{\pgfqpoint{1.305cm}{1.468cm}}
\pgfpathcurveto{\pgfqpoint{1.28cm}{1.494cm}}{\pgfqpoint{1.245cm}{1.508cm}}{\pgfqpoint{1.209cm}{1.508cm}}
\pgfpathcurveto{\pgfqpoint{1.172cm}{1.508cm}}{\pgfqpoint{1.138cm}{1.494cm}}{\pgfqpoint{1.112cm}{1.468cm}}
\pgfpathcurveto{\pgfqpoint{1.087cm}{1.442cm}}{\pgfqpoint{1.072cm}{1.408cm}}{\pgfqpoint{1.072cm}{1.371cm}}
\pgfpathcurveto{\pgfqpoint{1.072cm}{1.335cm}}{\pgfqpoint{1.087cm}{1.3cm}}{\pgfqpoint{1.112cm}{1.274cm}}
\pgfpathcurveto{\pgfqpoint{1.138cm}{1.249cm}}{\pgfqpoint{1.172cm}{1.234cm}}{\pgfqpoint{1.209cm}{1.234cm}}
\pgfpathcurveto{\pgfqpoint{1.245cm}{1.234cm}}{\pgfqpoint{1.28cm}{1.249cm}}{\pgfqpoint{1.305cm}{1.274cm}}
\pgfpathcurveto{\pgfqpoint{1.331cm}{1.3cm}}{\pgfqpoint{1.345cm}{1.335cm}}{\pgfqpoint{1.345cm}{1.371cm}}
\pgfusepath{fill}
\begin{pgfscope}
\pgfsetdash{}{0cm}
\pgfsetlinewidth{0.818mm}
\pgfsetroundcap
\pgfsetmiterlimit{4.0}
\pgfpathmoveto{\pgfqpoint{0.682cm}{0.671cm}}
\pgfpathlineto{\pgfqpoint{0.682cm}{0.042cm}}
\pgfusepath{stroke}
\end{pgfscope}
\end{pgfscope}
\end{pgfscope}
\end{pgfscope}
\end{tikzpicture}}} + \phi + \psi)\\
&\quad= - 3X^{\!\resizebox{!}{.8em}{
\begin{tikzpicture}
\pgfpathmoveto{\pgfqpoint{0cm}{-0.035cm}}
\pgfpathlineto{\pgfqpoint{1.976cm}{-0.035cm}}
\pgfpathlineto{\pgfqpoint{1.976cm}{1.94cm}}
\pgfpathlineto{\pgfqpoint{0cm}{1.94cm}}
\pgfpathclose
\pgfusepath{clip}
\begin{pgfscope}
\begin{pgfscope}
\pgfpathmoveto{\pgfqpoint{0cm}{-0.035cm}}
\pgfpathlineto{\pgfqpoint{1.976cm}{-0.035cm}}
\pgfpathlineto{\pgfqpoint{1.976cm}{1.94cm}}
\pgfpathlineto{\pgfqpoint{0cm}{1.94cm}}
\pgfpathclose
\pgfusepath{clip}
\begin{pgfscope}
\begin{pgfscope}
\pgfsetdash{}{0cm}
\pgfsetlinewidth{0.818mm}
\pgfsetroundcap
\pgfsetroundjoin
\pgfsetmiterlimit{7.0}
\definecolor{eps2pgf_color}{gray}{0}\pgfsetstrokecolor{eps2pgf_color}\pgfsetfillcolor{eps2pgf_color}
\pgfpathmoveto{\pgfqpoint{0.117cm}{1.815cm}}
\pgfpathlineto{\pgfqpoint{0.682cm}{1.065cm}}
\pgfpathlineto{\pgfqpoint{1.246cm}{1.815cm}}
\pgfusepath{stroke}
\end{pgfscope}
\definecolor{eps2pgf_color}{gray}{0}\pgfsetstrokecolor{eps2pgf_color}\pgfsetfillcolor{eps2pgf_color}
\pgfpathmoveto{\pgfqpoint{0.273cm}{1.789cm}}
\pgfpathcurveto{\pgfqpoint{0.273cm}{1.825cm}}{\pgfqpoint{0.259cm}{1.86cm}}{\pgfqpoint{0.233cm}{1.886cm}}
\pgfpathcurveto{\pgfqpoint{0.207cm}{1.912cm}}{\pgfqpoint{0.173cm}{1.926cm}}{\pgfqpoint{0.137cm}{1.926cm}}
\pgfpathcurveto{\pgfqpoint{0.1cm}{1.926cm}}{\pgfqpoint{0.066cm}{1.912cm}}{\pgfqpoint{0.04cm}{1.886cm}}
\pgfpathcurveto{\pgfqpoint{0.014cm}{1.86cm}}{\pgfqpoint{0cm}{1.825cm}}{\pgfqpoint{0cm}{1.789cm}}
\pgfpathcurveto{\pgfqpoint{0cm}{1.753cm}}{\pgfqpoint{0.014cm}{1.718cm}}{\pgfqpoint{0.04cm}{1.692cm}}
\pgfpathcurveto{\pgfqpoint{0.066cm}{1.667cm}}{\pgfqpoint{0.1cm}{1.652cm}}{\pgfqpoint{0.137cm}{1.652cm}}
\pgfpathcurveto{\pgfqpoint{0.173cm}{1.652cm}}{\pgfqpoint{0.207cm}{1.667cm}}{\pgfqpoint{0.233cm}{1.692cm}}
\pgfpathcurveto{\pgfqpoint{0.259cm}{1.718cm}}{\pgfqpoint{0.273cm}{1.753cm}}{\pgfqpoint{0.273cm}{1.789cm}}
\pgfusepath{fill}
\begin{pgfscope}
\pgfsetdash{}{0cm}
\pgfsetlinewidth{0.818mm}
\pgfsetmiterlimit{7.0}
\pgfpathmoveto{\pgfqpoint{0.682cm}{1.065cm}}
\pgfpathlineto{\pgfqpoint{0.679cm}{1.812cm}}
\pgfusepath{stroke}
\end{pgfscope}
\pgfpathmoveto{\pgfqpoint{0.815cm}{1.793cm}}
\pgfpathcurveto{\pgfqpoint{0.815cm}{1.829cm}}{\pgfqpoint{0.801cm}{1.864cm}}{\pgfqpoint{0.775cm}{1.89cm}}
\pgfpathcurveto{\pgfqpoint{0.75cm}{1.915cm}}{\pgfqpoint{0.715cm}{1.93cm}}{\pgfqpoint{0.679cm}{1.93cm}}
\pgfpathcurveto{\pgfqpoint{0.643cm}{1.93cm}}{\pgfqpoint{0.608cm}{1.915cm}}{\pgfqpoint{0.582cm}{1.89cm}}
\pgfpathcurveto{\pgfqpoint{0.557cm}{1.864cm}}{\pgfqpoint{0.542cm}{1.829cm}}{\pgfqpoint{0.542cm}{1.793cm}}
\pgfpathcurveto{\pgfqpoint{0.542cm}{1.756cm}}{\pgfqpoint{0.557cm}{1.722cm}}{\pgfqpoint{0.582cm}{1.696cm}}
\pgfpathcurveto{\pgfqpoint{0.608cm}{1.67cm}}{\pgfqpoint{0.643cm}{1.656cm}}{\pgfqpoint{0.679cm}{1.656cm}}
\pgfpathcurveto{\pgfqpoint{0.715cm}{1.656cm}}{\pgfqpoint{0.75cm}{1.67cm}}{\pgfqpoint{0.775cm}{1.696cm}}
\pgfpathcurveto{\pgfqpoint{0.801cm}{1.722cm}}{\pgfqpoint{0.815cm}{1.756cm}}{\pgfqpoint{0.815cm}{1.793cm}}
\pgfusepath{fill}
\pgfpathmoveto{\pgfqpoint{1.345cm}{1.765cm}}
\pgfpathcurveto{\pgfqpoint{1.345cm}{1.801cm}}{\pgfqpoint{1.331cm}{1.836cm}}{\pgfqpoint{1.305cm}{1.862cm}}
\pgfpathcurveto{\pgfqpoint{1.28cm}{1.887cm}}{\pgfqpoint{1.245cm}{1.902cm}}{\pgfqpoint{1.209cm}{1.902cm}}
\pgfpathcurveto{\pgfqpoint{1.172cm}{1.902cm}}{\pgfqpoint{1.138cm}{1.887cm}}{\pgfqpoint{1.112cm}{1.862cm}}
\pgfpathcurveto{\pgfqpoint{1.087cm}{1.836cm}}{\pgfqpoint{1.072cm}{1.801cm}}{\pgfqpoint{1.072cm}{1.765cm}}
\pgfpathcurveto{\pgfqpoint{1.072cm}{1.728cm}}{\pgfqpoint{1.087cm}{1.694cm}}{\pgfqpoint{1.112cm}{1.668cm}}
\pgfpathcurveto{\pgfqpoint{1.138cm}{1.642cm}}{\pgfqpoint{1.172cm}{1.628cm}}{\pgfqpoint{1.209cm}{1.628cm}}
\pgfpathcurveto{\pgfqpoint{1.245cm}{1.628cm}}{\pgfqpoint{1.28cm}{1.642cm}}{\pgfqpoint{1.305cm}{1.668cm}}
\pgfpathcurveto{\pgfqpoint{1.331cm}{1.694cm}}{\pgfqpoint{1.345cm}{1.728cm}}{\pgfqpoint{1.345cm}{1.765cm}}
\pgfusepath{fill}
\begin{pgfscope}
\pgfsetdash{}{0cm}
\pgfsetlinewidth{0.818mm}
\pgfsetroundcap
\pgfsetroundjoin
\pgfsetmiterlimit{7.0}
\pgfpathmoveto{\pgfqpoint{0.682cm}{1.065cm}}
\pgfpathlineto{\pgfqpoint{1.246cm}{0.315cm}}
\pgfpathlineto{\pgfqpoint{1.811cm}{1.065cm}}
\pgfusepath{stroke}
\end{pgfscope}
\pgfpathmoveto{\pgfqpoint{1.948cm}{1.065cm}}
\pgfpathcurveto{\pgfqpoint{1.948cm}{1.101cm}}{\pgfqpoint{1.933cm}{1.136cm}}{\pgfqpoint{1.907cm}{1.162cm}}
\pgfpathcurveto{\pgfqpoint{1.882cm}{1.187cm}}{\pgfqpoint{1.847cm}{1.202cm}}{\pgfqpoint{1.811cm}{1.202cm}}
\pgfpathcurveto{\pgfqpoint{1.775cm}{1.202cm}}{\pgfqpoint{1.74cm}{1.187cm}}{\pgfqpoint{1.714cm}{1.162cm}}
\pgfpathcurveto{\pgfqpoint{1.689cm}{1.136cm}}{\pgfqpoint{1.674cm}{1.101cm}}{\pgfqpoint{1.674cm}{1.065cm}}
\pgfpathcurveto{\pgfqpoint{1.674cm}{1.029cm}}{\pgfqpoint{1.689cm}{0.994cm}}{\pgfqpoint{1.714cm}{0.968cm}}
\pgfpathcurveto{\pgfqpoint{1.74cm}{0.942cm}}{\pgfqpoint{1.775cm}{0.928cm}}{\pgfqpoint{1.811cm}{0.928cm}}
\pgfpathcurveto{\pgfqpoint{1.847cm}{0.928cm}}{\pgfqpoint{1.882cm}{0.942cm}}{\pgfqpoint{1.907cm}{0.968cm}}
\pgfpathcurveto{\pgfqpoint{1.933cm}{0.994cm}}{\pgfqpoint{1.948cm}{1.029cm}}{\pgfqpoint{1.948cm}{1.065cm}}
\pgfusepath{fill}
\begin{pgfscope}
\pgfsetdash{}{0cm}
\pgfsetlinewidth{0.818mm}
\pgfsetmiterlimit{7.0}
\pgfpathmoveto{\pgfqpoint{1.246cm}{0.315cm}}
\pgfpathlineto{\pgfqpoint{1.244cm}{1.061cm}}
\pgfusepath{stroke}
\end{pgfscope}
\pgfpathmoveto{\pgfqpoint{1.38cm}{1.065cm}}
\pgfpathcurveto{\pgfqpoint{1.38cm}{1.101cm}}{\pgfqpoint{1.366cm}{1.136cm}}{\pgfqpoint{1.34cm}{1.162cm}}
\pgfpathcurveto{\pgfqpoint{1.315cm}{1.187cm}}{\pgfqpoint{1.28cm}{1.202cm}}{\pgfqpoint{1.244cm}{1.202cm}}
\pgfpathcurveto{\pgfqpoint{1.207cm}{1.202cm}}{\pgfqpoint{1.173cm}{1.187cm}}{\pgfqpoint{1.147cm}{1.162cm}}
\pgfpathcurveto{\pgfqpoint{1.121cm}{1.136cm}}{\pgfqpoint{1.107cm}{1.101cm}}{\pgfqpoint{1.107cm}{1.065cm}}
\pgfpathcurveto{\pgfqpoint{1.107cm}{1.029cm}}{\pgfqpoint{1.121cm}{0.994cm}}{\pgfqpoint{1.147cm}{0.968cm}}
\pgfpathcurveto{\pgfqpoint{1.173cm}{0.942cm}}{\pgfqpoint{1.207cm}{0.928cm}}{\pgfqpoint{1.244cm}{0.928cm}}
\pgfpathcurveto{\pgfqpoint{1.28cm}{0.928cm}}{\pgfqpoint{1.315cm}{0.942cm}}{\pgfqpoint{1.34cm}{0.968cm}}
\pgfpathcurveto{\pgfqpoint{1.366cm}{0.994cm}}{\pgfqpoint{1.38cm}{1.029cm}}{\pgfqpoint{1.38cm}{1.065cm}}
\pgfusepath{fill}
\begin{pgfscope}
\pgfsetdash{}{0cm}
\pgfsetlinewidth{0.818mm}
\pgfsetmiterlimit{4.0}
\pgfpathmoveto{\pgfqpoint{1.383cm}{0.178cm}}
\pgfpathcurveto{\pgfqpoint{1.383cm}{0.214cm}}{\pgfqpoint{1.369cm}{0.249cm}}{\pgfqpoint{1.343cm}{0.275cm}}
\pgfpathcurveto{\pgfqpoint{1.317cm}{0.3cm}}{\pgfqpoint{1.283cm}{0.315cm}}{\pgfqpoint{1.246cm}{0.315cm}}
\pgfpathcurveto{\pgfqpoint{1.21cm}{0.315cm}}{\pgfqpoint{1.175cm}{0.3cm}}{\pgfqpoint{1.15cm}{0.275cm}}
\pgfpathcurveto{\pgfqpoint{1.124cm}{0.249cm}}{\pgfqpoint{1.11cm}{0.214cm}}{\pgfqpoint{1.11cm}{0.178cm}}
\pgfpathcurveto{\pgfqpoint{1.11cm}{0.141cm}}{\pgfqpoint{1.124cm}{0.107cm}}{\pgfqpoint{1.15cm}{0.081cm}}
\pgfpathcurveto{\pgfqpoint{1.175cm}{0.055cm}}{\pgfqpoint{1.21cm}{0.041cm}}{\pgfqpoint{1.246cm}{0.041cm}}
\pgfpathcurveto{\pgfqpoint{1.283cm}{0.041cm}}{\pgfqpoint{1.317cm}{0.055cm}}{\pgfqpoint{1.343cm}{0.081cm}}
\pgfpathcurveto{\pgfqpoint{1.369cm}{0.107cm}}{\pgfqpoint{1.383cm}{0.141cm}}{\pgfqpoint{1.383cm}{0.178cm}}
\pgfusepath{stroke}
\end{pgfscope}
\end{pgfscope}
\end{pgfscope}
\end{pgfscope}
\end{tikzpicture}}}+3\llbracket X^2 \rrbracket\circ\vartheta-9( - X^{\!\resizebox{0.6em}{!}{
\begin{tikzpicture}
\pgfpathmoveto{\pgfqpoint{0cm}{-0.035cm}}
\pgfpathlineto{\pgfqpoint{1.376cm}{-0.035cm}}
\pgfpathlineto{\pgfqpoint{1.376cm}{1.552cm}}
\pgfpathlineto{\pgfqpoint{0cm}{1.552cm}}
\pgfpathclose
\pgfusepath{clip}
\begin{pgfscope}
\begin{pgfscope}
\pgfpathmoveto{\pgfqpoint{0cm}{-0.035cm}}
\pgfpathlineto{\pgfqpoint{1.376cm}{-0.035cm}}
\pgfpathlineto{\pgfqpoint{1.376cm}{1.552cm}}
\pgfpathlineto{\pgfqpoint{0cm}{1.552cm}}
\pgfpathclose
\pgfusepath{clip}
\begin{pgfscope}
\begin{pgfscope}
\pgfsetdash{}{0cm}
\pgfsetlinewidth{0.818mm}
\pgfsetroundcap
\pgfsetroundjoin
\pgfsetmiterlimit{7.0}
\definecolor{eps2pgf_color}{gray}{0}\pgfsetstrokecolor{eps2pgf_color}\pgfsetfillcolor{eps2pgf_color}
\pgfpathmoveto{\pgfqpoint{0.117cm}{1.421cm}}
\pgfpathlineto{\pgfqpoint{0.682cm}{0.671cm}}
\pgfpathlineto{\pgfqpoint{1.246cm}{1.421cm}}
\pgfusepath{stroke}
\end{pgfscope}
\definecolor{eps2pgf_color}{gray}{0}\pgfsetstrokecolor{eps2pgf_color}\pgfsetfillcolor{eps2pgf_color}
\pgfpathmoveto{\pgfqpoint{0.273cm}{1.395cm}}
\pgfpathcurveto{\pgfqpoint{0.273cm}{1.432cm}}{\pgfqpoint{0.259cm}{1.467cm}}{\pgfqpoint{0.233cm}{1.492cm}}
\pgfpathcurveto{\pgfqpoint{0.207cm}{1.518cm}}{\pgfqpoint{0.173cm}{1.532cm}}{\pgfqpoint{0.137cm}{1.532cm}}
\pgfpathcurveto{\pgfqpoint{0.1cm}{1.532cm}}{\pgfqpoint{0.066cm}{1.518cm}}{\pgfqpoint{0.04cm}{1.492cm}}
\pgfpathcurveto{\pgfqpoint{0.014cm}{1.467cm}}{\pgfqpoint{0cm}{1.432cm}}{\pgfqpoint{0cm}{1.395cm}}
\pgfpathcurveto{\pgfqpoint{0cm}{1.359cm}}{\pgfqpoint{0.014cm}{1.324cm}}{\pgfqpoint{0.04cm}{1.299cm}}
\pgfpathcurveto{\pgfqpoint{0.066cm}{1.273cm}}{\pgfqpoint{0.1cm}{1.258cm}}{\pgfqpoint{0.137cm}{1.258cm}}
\pgfpathcurveto{\pgfqpoint{0.173cm}{1.258cm}}{\pgfqpoint{0.207cm}{1.273cm}}{\pgfqpoint{0.233cm}{1.299cm}}
\pgfpathcurveto{\pgfqpoint{0.259cm}{1.324cm}}{\pgfqpoint{0.273cm}{1.359cm}}{\pgfqpoint{0.273cm}{1.395cm}}
\pgfusepath{fill}
\begin{pgfscope}
\pgfsetdash{}{0cm}
\pgfsetlinewidth{0.818mm}
\pgfsetmiterlimit{7.0}
\pgfpathmoveto{\pgfqpoint{0.682cm}{0.671cm}}
\pgfpathlineto{\pgfqpoint{0.679cm}{1.418cm}}
\pgfusepath{stroke}
\end{pgfscope}
\pgfpathmoveto{\pgfqpoint{0.815cm}{1.399cm}}
\pgfpathcurveto{\pgfqpoint{0.815cm}{1.435cm}}{\pgfqpoint{0.801cm}{1.47cm}}{\pgfqpoint{0.775cm}{1.496cm}}
\pgfpathcurveto{\pgfqpoint{0.75cm}{1.521cm}}{\pgfqpoint{0.715cm}{1.536cm}}{\pgfqpoint{0.679cm}{1.536cm}}
\pgfpathcurveto{\pgfqpoint{0.643cm}{1.536cm}}{\pgfqpoint{0.608cm}{1.521cm}}{\pgfqpoint{0.582cm}{1.496cm}}
\pgfpathcurveto{\pgfqpoint{0.557cm}{1.47cm}}{\pgfqpoint{0.542cm}{1.435cm}}{\pgfqpoint{0.542cm}{1.399cm}}
\pgfpathcurveto{\pgfqpoint{0.542cm}{1.363cm}}{\pgfqpoint{0.557cm}{1.328cm}}{\pgfqpoint{0.582cm}{1.302cm}}
\pgfpathcurveto{\pgfqpoint{0.608cm}{1.276cm}}{\pgfqpoint{0.643cm}{1.262cm}}{\pgfqpoint{0.679cm}{1.262cm}}
\pgfpathcurveto{\pgfqpoint{0.715cm}{1.262cm}}{\pgfqpoint{0.75cm}{1.276cm}}{\pgfqpoint{0.775cm}{1.302cm}}
\pgfpathcurveto{\pgfqpoint{0.801cm}{1.328cm}}{\pgfqpoint{0.815cm}{1.363cm}}{\pgfqpoint{0.815cm}{1.399cm}}
\pgfusepath{fill}
\pgfpathmoveto{\pgfqpoint{1.345cm}{1.371cm}}
\pgfpathcurveto{\pgfqpoint{1.345cm}{1.408cm}}{\pgfqpoint{1.331cm}{1.442cm}}{\pgfqpoint{1.305cm}{1.468cm}}
\pgfpathcurveto{\pgfqpoint{1.28cm}{1.494cm}}{\pgfqpoint{1.245cm}{1.508cm}}{\pgfqpoint{1.209cm}{1.508cm}}
\pgfpathcurveto{\pgfqpoint{1.172cm}{1.508cm}}{\pgfqpoint{1.138cm}{1.494cm}}{\pgfqpoint{1.112cm}{1.468cm}}
\pgfpathcurveto{\pgfqpoint{1.087cm}{1.442cm}}{\pgfqpoint{1.072cm}{1.408cm}}{\pgfqpoint{1.072cm}{1.371cm}}
\pgfpathcurveto{\pgfqpoint{1.072cm}{1.335cm}}{\pgfqpoint{1.087cm}{1.3cm}}{\pgfqpoint{1.112cm}{1.274cm}}
\pgfpathcurveto{\pgfqpoint{1.138cm}{1.249cm}}{\pgfqpoint{1.172cm}{1.234cm}}{\pgfqpoint{1.209cm}{1.234cm}}
\pgfpathcurveto{\pgfqpoint{1.245cm}{1.234cm}}{\pgfqpoint{1.28cm}{1.249cm}}{\pgfqpoint{1.305cm}{1.274cm}}
\pgfpathcurveto{\pgfqpoint{1.331cm}{1.3cm}}{\pgfqpoint{1.345cm}{1.335cm}}{\pgfqpoint{1.345cm}{1.371cm}}
\pgfusepath{fill}
\begin{pgfscope}
\pgfsetdash{}{0cm}
\pgfsetlinewidth{0.818mm}
\pgfsetroundcap
\pgfsetmiterlimit{4.0}
\pgfpathmoveto{\pgfqpoint{0.682cm}{0.671cm}}
\pgfpathlineto{\pgfqpoint{0.682cm}{0.042cm}}
\pgfusepath{stroke}
\end{pgfscope}
\end{pgfscope}
\end{pgfscope}
\end{pgfscope}
\end{tikzpicture}}} + \phi + \psi)(X^{\!\resizebox{0.6em}{!}{
\begin{tikzpicture}
\pgfpathmoveto{\pgfqpoint{0cm}{0cm}}
\pgfpathlineto{\pgfqpoint{1.376cm}{0cm}}
\pgfpathlineto{\pgfqpoint{1.376cm}{1.588cm}}
\pgfpathlineto{\pgfqpoint{0cm}{1.588cm}}
\pgfpathclose
\pgfusepath{clip}
\begin{pgfscope}
\begin{pgfscope}
\pgfpathmoveto{\pgfqpoint{0cm}{0cm}}
\pgfpathlineto{\pgfqpoint{1.376cm}{0cm}}
\pgfpathlineto{\pgfqpoint{1.376cm}{1.588cm}}
\pgfpathlineto{\pgfqpoint{0cm}{1.588cm}}
\pgfpathclose
\pgfusepath{clip}
\begin{pgfscope}
\begin{pgfscope}
\definecolor{eps2pgf_color}{gray}{0.976471}\pgfsetstrokecolor{eps2pgf_color}\pgfsetfillcolor{eps2pgf_color}
\pgfpathmoveto{\pgfqpoint{0cm}{0cm}}
\pgfpathlineto{\pgfqpoint{1.376cm}{0cm}}
\pgfpathlineto{\pgfqpoint{1.376cm}{1.588cm}}
\pgfpathlineto{\pgfqpoint{0cm}{1.588cm}}
\pgfpathclose
\pgfusepath{fill}
\end{pgfscope}
\begin{pgfscope}
\pgfsetdash{}{0cm}
\pgfsetlinewidth{0.818mm}
\pgfsetroundcap
\pgfsetroundjoin
\pgfsetmiterlimit{7.0}
\definecolor{eps2pgf_color}{gray}{0}\pgfsetstrokecolor{eps2pgf_color}\pgfsetfillcolor{eps2pgf_color}
\pgfpathmoveto{\pgfqpoint{0.117cm}{1.476cm}}
\pgfpathlineto{\pgfqpoint{0.682cm}{0.726cm}}
\pgfpathlineto{\pgfqpoint{1.246cm}{1.476cm}}
\pgfusepath{stroke}
\end{pgfscope}
\definecolor{eps2pgf_color}{gray}{0}\pgfsetstrokecolor{eps2pgf_color}\pgfsetfillcolor{eps2pgf_color}
\pgfpathmoveto{\pgfqpoint{0.273cm}{1.451cm}}
\pgfpathcurveto{\pgfqpoint{0.273cm}{1.487cm}}{\pgfqpoint{0.259cm}{1.522cm}}{\pgfqpoint{0.233cm}{1.547cm}}
\pgfpathcurveto{\pgfqpoint{0.207cm}{1.573cm}}{\pgfqpoint{0.173cm}{1.588cm}}{\pgfqpoint{0.137cm}{1.588cm}}
\pgfpathcurveto{\pgfqpoint{0.1cm}{1.588cm}}{\pgfqpoint{0.066cm}{1.573cm}}{\pgfqpoint{0.04cm}{1.547cm}}
\pgfpathcurveto{\pgfqpoint{0.014cm}{1.522cm}}{\pgfqpoint{0cm}{1.487cm}}{\pgfqpoint{0cm}{1.451cm}}
\pgfpathcurveto{\pgfqpoint{0cm}{1.414cm}}{\pgfqpoint{0.014cm}{1.379cm}}{\pgfqpoint{0.04cm}{1.354cm}}
\pgfpathcurveto{\pgfqpoint{0.066cm}{1.328cm}}{\pgfqpoint{0.1cm}{1.314cm}}{\pgfqpoint{0.137cm}{1.314cm}}
\pgfpathcurveto{\pgfqpoint{0.173cm}{1.314cm}}{\pgfqpoint{0.207cm}{1.328cm}}{\pgfqpoint{0.233cm}{1.354cm}}
\pgfpathcurveto{\pgfqpoint{0.259cm}{1.379cm}}{\pgfqpoint{0.273cm}{1.414cm}}{\pgfqpoint{0.273cm}{1.451cm}}
\pgfusepath{fill}
\pgfpathmoveto{\pgfqpoint{1.345cm}{1.426cm}}
\pgfpathcurveto{\pgfqpoint{1.345cm}{1.463cm}}{\pgfqpoint{1.331cm}{1.497cm}}{\pgfqpoint{1.305cm}{1.523cm}}
\pgfpathcurveto{\pgfqpoint{1.28cm}{1.549cm}}{\pgfqpoint{1.245cm}{1.563cm}}{\pgfqpoint{1.209cm}{1.563cm}}
\pgfpathcurveto{\pgfqpoint{1.172cm}{1.563cm}}{\pgfqpoint{1.138cm}{1.549cm}}{\pgfqpoint{1.112cm}{1.523cm}}
\pgfpathcurveto{\pgfqpoint{1.087cm}{1.497cm}}{\pgfqpoint{1.072cm}{1.463cm}}{\pgfqpoint{1.072cm}{1.426cm}}
\pgfpathcurveto{\pgfqpoint{1.072cm}{1.39cm}}{\pgfqpoint{1.087cm}{1.355cm}}{\pgfqpoint{1.112cm}{1.329cm}}
\pgfpathcurveto{\pgfqpoint{1.138cm}{1.304cm}}{\pgfqpoint{1.172cm}{1.289cm}}{\pgfqpoint{1.209cm}{1.289cm}}
\pgfpathcurveto{\pgfqpoint{1.245cm}{1.289cm}}{\pgfqpoint{1.28cm}{1.304cm}}{\pgfqpoint{1.305cm}{1.329cm}}
\pgfpathcurveto{\pgfqpoint{1.331cm}{1.355cm}}{\pgfqpoint{1.345cm}{1.39cm}}{\pgfqpoint{1.345cm}{1.426cm}}
\pgfusepath{fill}
\begin{pgfscope}
\pgfsetdash{}{0cm}
\pgfsetlinewidth{0.818mm}
\pgfsetroundcap
\pgfsetmiterlimit{4.0}
\pgfpathmoveto{\pgfqpoint{0.682cm}{0.726cm}}
\pgfpathlineto{\pgfqpoint{0.682cm}{0.097cm}}
\pgfusepath{stroke}
\end{pgfscope}
\end{pgfscope}
\end{pgfscope}
\end{pgfscope}
\end{tikzpicture}}}\circ\llbracket X^2 \rrbracket)\\
&\qquad\qquad-9\mathrm{com}(-X^{\!\resizebox{0.6em}{!}{
\begin{tikzpicture}
\pgfpathmoveto{\pgfqpoint{0cm}{-0.035cm}}
\pgfpathlineto{\pgfqpoint{1.376cm}{-0.035cm}}
\pgfpathlineto{\pgfqpoint{1.376cm}{1.552cm}}
\pgfpathlineto{\pgfqpoint{0cm}{1.552cm}}
\pgfpathclose
\pgfusepath{clip}
\begin{pgfscope}
\begin{pgfscope}
\pgfpathmoveto{\pgfqpoint{0cm}{-0.035cm}}
\pgfpathlineto{\pgfqpoint{1.376cm}{-0.035cm}}
\pgfpathlineto{\pgfqpoint{1.376cm}{1.552cm}}
\pgfpathlineto{\pgfqpoint{0cm}{1.552cm}}
\pgfpathclose
\pgfusepath{clip}
\begin{pgfscope}
\begin{pgfscope}
\pgfsetdash{}{0cm}
\pgfsetlinewidth{0.818mm}
\pgfsetroundcap
\pgfsetroundjoin
\pgfsetmiterlimit{7.0}
\definecolor{eps2pgf_color}{gray}{0}\pgfsetstrokecolor{eps2pgf_color}\pgfsetfillcolor{eps2pgf_color}
\pgfpathmoveto{\pgfqpoint{0.117cm}{1.421cm}}
\pgfpathlineto{\pgfqpoint{0.682cm}{0.671cm}}
\pgfpathlineto{\pgfqpoint{1.246cm}{1.421cm}}
\pgfusepath{stroke}
\end{pgfscope}
\definecolor{eps2pgf_color}{gray}{0}\pgfsetstrokecolor{eps2pgf_color}\pgfsetfillcolor{eps2pgf_color}
\pgfpathmoveto{\pgfqpoint{0.273cm}{1.395cm}}
\pgfpathcurveto{\pgfqpoint{0.273cm}{1.432cm}}{\pgfqpoint{0.259cm}{1.467cm}}{\pgfqpoint{0.233cm}{1.492cm}}
\pgfpathcurveto{\pgfqpoint{0.207cm}{1.518cm}}{\pgfqpoint{0.173cm}{1.532cm}}{\pgfqpoint{0.137cm}{1.532cm}}
\pgfpathcurveto{\pgfqpoint{0.1cm}{1.532cm}}{\pgfqpoint{0.066cm}{1.518cm}}{\pgfqpoint{0.04cm}{1.492cm}}
\pgfpathcurveto{\pgfqpoint{0.014cm}{1.467cm}}{\pgfqpoint{0cm}{1.432cm}}{\pgfqpoint{0cm}{1.395cm}}
\pgfpathcurveto{\pgfqpoint{0cm}{1.359cm}}{\pgfqpoint{0.014cm}{1.324cm}}{\pgfqpoint{0.04cm}{1.299cm}}
\pgfpathcurveto{\pgfqpoint{0.066cm}{1.273cm}}{\pgfqpoint{0.1cm}{1.258cm}}{\pgfqpoint{0.137cm}{1.258cm}}
\pgfpathcurveto{\pgfqpoint{0.173cm}{1.258cm}}{\pgfqpoint{0.207cm}{1.273cm}}{\pgfqpoint{0.233cm}{1.299cm}}
\pgfpathcurveto{\pgfqpoint{0.259cm}{1.324cm}}{\pgfqpoint{0.273cm}{1.359cm}}{\pgfqpoint{0.273cm}{1.395cm}}
\pgfusepath{fill}
\begin{pgfscope}
\pgfsetdash{}{0cm}
\pgfsetlinewidth{0.818mm}
\pgfsetmiterlimit{7.0}
\pgfpathmoveto{\pgfqpoint{0.682cm}{0.671cm}}
\pgfpathlineto{\pgfqpoint{0.679cm}{1.418cm}}
\pgfusepath{stroke}
\end{pgfscope}
\pgfpathmoveto{\pgfqpoint{0.815cm}{1.399cm}}
\pgfpathcurveto{\pgfqpoint{0.815cm}{1.435cm}}{\pgfqpoint{0.801cm}{1.47cm}}{\pgfqpoint{0.775cm}{1.496cm}}
\pgfpathcurveto{\pgfqpoint{0.75cm}{1.521cm}}{\pgfqpoint{0.715cm}{1.536cm}}{\pgfqpoint{0.679cm}{1.536cm}}
\pgfpathcurveto{\pgfqpoint{0.643cm}{1.536cm}}{\pgfqpoint{0.608cm}{1.521cm}}{\pgfqpoint{0.582cm}{1.496cm}}
\pgfpathcurveto{\pgfqpoint{0.557cm}{1.47cm}}{\pgfqpoint{0.542cm}{1.435cm}}{\pgfqpoint{0.542cm}{1.399cm}}
\pgfpathcurveto{\pgfqpoint{0.542cm}{1.363cm}}{\pgfqpoint{0.557cm}{1.328cm}}{\pgfqpoint{0.582cm}{1.302cm}}
\pgfpathcurveto{\pgfqpoint{0.608cm}{1.276cm}}{\pgfqpoint{0.643cm}{1.262cm}}{\pgfqpoint{0.679cm}{1.262cm}}
\pgfpathcurveto{\pgfqpoint{0.715cm}{1.262cm}}{\pgfqpoint{0.75cm}{1.276cm}}{\pgfqpoint{0.775cm}{1.302cm}}
\pgfpathcurveto{\pgfqpoint{0.801cm}{1.328cm}}{\pgfqpoint{0.815cm}{1.363cm}}{\pgfqpoint{0.815cm}{1.399cm}}
\pgfusepath{fill}
\pgfpathmoveto{\pgfqpoint{1.345cm}{1.371cm}}
\pgfpathcurveto{\pgfqpoint{1.345cm}{1.408cm}}{\pgfqpoint{1.331cm}{1.442cm}}{\pgfqpoint{1.305cm}{1.468cm}}
\pgfpathcurveto{\pgfqpoint{1.28cm}{1.494cm}}{\pgfqpoint{1.245cm}{1.508cm}}{\pgfqpoint{1.209cm}{1.508cm}}
\pgfpathcurveto{\pgfqpoint{1.172cm}{1.508cm}}{\pgfqpoint{1.138cm}{1.494cm}}{\pgfqpoint{1.112cm}{1.468cm}}
\pgfpathcurveto{\pgfqpoint{1.087cm}{1.442cm}}{\pgfqpoint{1.072cm}{1.408cm}}{\pgfqpoint{1.072cm}{1.371cm}}
\pgfpathcurveto{\pgfqpoint{1.072cm}{1.335cm}}{\pgfqpoint{1.087cm}{1.3cm}}{\pgfqpoint{1.112cm}{1.274cm}}
\pgfpathcurveto{\pgfqpoint{1.138cm}{1.249cm}}{\pgfqpoint{1.172cm}{1.234cm}}{\pgfqpoint{1.209cm}{1.234cm}}
\pgfpathcurveto{\pgfqpoint{1.245cm}{1.234cm}}{\pgfqpoint{1.28cm}{1.249cm}}{\pgfqpoint{1.305cm}{1.274cm}}
\pgfpathcurveto{\pgfqpoint{1.331cm}{1.3cm}}{\pgfqpoint{1.345cm}{1.335cm}}{\pgfqpoint{1.345cm}{1.371cm}}
\pgfusepath{fill}
\begin{pgfscope}
\pgfsetdash{}{0cm}
\pgfsetlinewidth{0.818mm}
\pgfsetroundcap
\pgfsetmiterlimit{4.0}
\pgfpathmoveto{\pgfqpoint{0.682cm}{0.671cm}}
\pgfpathlineto{\pgfqpoint{0.682cm}{0.042cm}}
\pgfusepath{stroke}
\end{pgfscope}
\end{pgfscope}
\end{pgfscope}
\end{pgfscope}
\end{tikzpicture}}} + \phi + \psi,X^{\!\resizebox{0.6em}{!}{
\begin{tikzpicture}
\pgfpathmoveto{\pgfqpoint{0cm}{0cm}}
\pgfpathlineto{\pgfqpoint{1.376cm}{0cm}}
\pgfpathlineto{\pgfqpoint{1.376cm}{1.588cm}}
\pgfpathlineto{\pgfqpoint{0cm}{1.588cm}}
\pgfpathclose
\pgfusepath{clip}
\begin{pgfscope}
\begin{pgfscope}
\pgfpathmoveto{\pgfqpoint{0cm}{0cm}}
\pgfpathlineto{\pgfqpoint{1.376cm}{0cm}}
\pgfpathlineto{\pgfqpoint{1.376cm}{1.588cm}}
\pgfpathlineto{\pgfqpoint{0cm}{1.588cm}}
\pgfpathclose
\pgfusepath{clip}
\begin{pgfscope}
\begin{pgfscope}
\definecolor{eps2pgf_color}{gray}{0.976471}\pgfsetstrokecolor{eps2pgf_color}\pgfsetfillcolor{eps2pgf_color}
\pgfpathmoveto{\pgfqpoint{0cm}{0cm}}
\pgfpathlineto{\pgfqpoint{1.376cm}{0cm}}
\pgfpathlineto{\pgfqpoint{1.376cm}{1.588cm}}
\pgfpathlineto{\pgfqpoint{0cm}{1.588cm}}
\pgfpathclose
\pgfusepath{fill}
\end{pgfscope}
\begin{pgfscope}
\pgfsetdash{}{0cm}
\pgfsetlinewidth{0.818mm}
\pgfsetroundcap
\pgfsetroundjoin
\pgfsetmiterlimit{7.0}
\definecolor{eps2pgf_color}{gray}{0}\pgfsetstrokecolor{eps2pgf_color}\pgfsetfillcolor{eps2pgf_color}
\pgfpathmoveto{\pgfqpoint{0.117cm}{1.476cm}}
\pgfpathlineto{\pgfqpoint{0.682cm}{0.726cm}}
\pgfpathlineto{\pgfqpoint{1.246cm}{1.476cm}}
\pgfusepath{stroke}
\end{pgfscope}
\definecolor{eps2pgf_color}{gray}{0}\pgfsetstrokecolor{eps2pgf_color}\pgfsetfillcolor{eps2pgf_color}
\pgfpathmoveto{\pgfqpoint{0.273cm}{1.451cm}}
\pgfpathcurveto{\pgfqpoint{0.273cm}{1.487cm}}{\pgfqpoint{0.259cm}{1.522cm}}{\pgfqpoint{0.233cm}{1.547cm}}
\pgfpathcurveto{\pgfqpoint{0.207cm}{1.573cm}}{\pgfqpoint{0.173cm}{1.588cm}}{\pgfqpoint{0.137cm}{1.588cm}}
\pgfpathcurveto{\pgfqpoint{0.1cm}{1.588cm}}{\pgfqpoint{0.066cm}{1.573cm}}{\pgfqpoint{0.04cm}{1.547cm}}
\pgfpathcurveto{\pgfqpoint{0.014cm}{1.522cm}}{\pgfqpoint{0cm}{1.487cm}}{\pgfqpoint{0cm}{1.451cm}}
\pgfpathcurveto{\pgfqpoint{0cm}{1.414cm}}{\pgfqpoint{0.014cm}{1.379cm}}{\pgfqpoint{0.04cm}{1.354cm}}
\pgfpathcurveto{\pgfqpoint{0.066cm}{1.328cm}}{\pgfqpoint{0.1cm}{1.314cm}}{\pgfqpoint{0.137cm}{1.314cm}}
\pgfpathcurveto{\pgfqpoint{0.173cm}{1.314cm}}{\pgfqpoint{0.207cm}{1.328cm}}{\pgfqpoint{0.233cm}{1.354cm}}
\pgfpathcurveto{\pgfqpoint{0.259cm}{1.379cm}}{\pgfqpoint{0.273cm}{1.414cm}}{\pgfqpoint{0.273cm}{1.451cm}}
\pgfusepath{fill}
\pgfpathmoveto{\pgfqpoint{1.345cm}{1.426cm}}
\pgfpathcurveto{\pgfqpoint{1.345cm}{1.463cm}}{\pgfqpoint{1.331cm}{1.497cm}}{\pgfqpoint{1.305cm}{1.523cm}}
\pgfpathcurveto{\pgfqpoint{1.28cm}{1.549cm}}{\pgfqpoint{1.245cm}{1.563cm}}{\pgfqpoint{1.209cm}{1.563cm}}
\pgfpathcurveto{\pgfqpoint{1.172cm}{1.563cm}}{\pgfqpoint{1.138cm}{1.549cm}}{\pgfqpoint{1.112cm}{1.523cm}}
\pgfpathcurveto{\pgfqpoint{1.087cm}{1.497cm}}{\pgfqpoint{1.072cm}{1.463cm}}{\pgfqpoint{1.072cm}{1.426cm}}
\pgfpathcurveto{\pgfqpoint{1.072cm}{1.39cm}}{\pgfqpoint{1.087cm}{1.355cm}}{\pgfqpoint{1.112cm}{1.329cm}}
\pgfpathcurveto{\pgfqpoint{1.138cm}{1.304cm}}{\pgfqpoint{1.172cm}{1.289cm}}{\pgfqpoint{1.209cm}{1.289cm}}
\pgfpathcurveto{\pgfqpoint{1.245cm}{1.289cm}}{\pgfqpoint{1.28cm}{1.304cm}}{\pgfqpoint{1.305cm}{1.329cm}}
\pgfpathcurveto{\pgfqpoint{1.331cm}{1.355cm}}{\pgfqpoint{1.345cm}{1.39cm}}{\pgfqpoint{1.345cm}{1.426cm}}
\pgfusepath{fill}
\begin{pgfscope}
\pgfsetdash{}{0cm}
\pgfsetlinewidth{0.818mm}
\pgfsetroundcap
\pgfsetmiterlimit{4.0}
\pgfpathmoveto{\pgfqpoint{0.682cm}{0.726cm}}
\pgfpathlineto{\pgfqpoint{0.682cm}{0.097cm}}
\pgfusepath{stroke}
\end{pgfscope}
\end{pgfscope}
\end{pgfscope}
\end{pgfscope}
\end{tikzpicture}}},\llbracket X^2 \rrbracket)+3b( - X^{\!\resizebox{0.6em}{!}{
\begin{tikzpicture}
\pgfpathmoveto{\pgfqpoint{0cm}{-0.035cm}}
\pgfpathlineto{\pgfqpoint{1.376cm}{-0.035cm}}
\pgfpathlineto{\pgfqpoint{1.376cm}{1.552cm}}
\pgfpathlineto{\pgfqpoint{0cm}{1.552cm}}
\pgfpathclose
\pgfusepath{clip}
\begin{pgfscope}
\begin{pgfscope}
\pgfpathmoveto{\pgfqpoint{0cm}{-0.035cm}}
\pgfpathlineto{\pgfqpoint{1.376cm}{-0.035cm}}
\pgfpathlineto{\pgfqpoint{1.376cm}{1.552cm}}
\pgfpathlineto{\pgfqpoint{0cm}{1.552cm}}
\pgfpathclose
\pgfusepath{clip}
\begin{pgfscope}
\begin{pgfscope}
\pgfsetdash{}{0cm}
\pgfsetlinewidth{0.818mm}
\pgfsetroundcap
\pgfsetroundjoin
\pgfsetmiterlimit{7.0}
\definecolor{eps2pgf_color}{gray}{0}\pgfsetstrokecolor{eps2pgf_color}\pgfsetfillcolor{eps2pgf_color}
\pgfpathmoveto{\pgfqpoint{0.117cm}{1.421cm}}
\pgfpathlineto{\pgfqpoint{0.682cm}{0.671cm}}
\pgfpathlineto{\pgfqpoint{1.246cm}{1.421cm}}
\pgfusepath{stroke}
\end{pgfscope}
\definecolor{eps2pgf_color}{gray}{0}\pgfsetstrokecolor{eps2pgf_color}\pgfsetfillcolor{eps2pgf_color}
\pgfpathmoveto{\pgfqpoint{0.273cm}{1.395cm}}
\pgfpathcurveto{\pgfqpoint{0.273cm}{1.432cm}}{\pgfqpoint{0.259cm}{1.467cm}}{\pgfqpoint{0.233cm}{1.492cm}}
\pgfpathcurveto{\pgfqpoint{0.207cm}{1.518cm}}{\pgfqpoint{0.173cm}{1.532cm}}{\pgfqpoint{0.137cm}{1.532cm}}
\pgfpathcurveto{\pgfqpoint{0.1cm}{1.532cm}}{\pgfqpoint{0.066cm}{1.518cm}}{\pgfqpoint{0.04cm}{1.492cm}}
\pgfpathcurveto{\pgfqpoint{0.014cm}{1.467cm}}{\pgfqpoint{0cm}{1.432cm}}{\pgfqpoint{0cm}{1.395cm}}
\pgfpathcurveto{\pgfqpoint{0cm}{1.359cm}}{\pgfqpoint{0.014cm}{1.324cm}}{\pgfqpoint{0.04cm}{1.299cm}}
\pgfpathcurveto{\pgfqpoint{0.066cm}{1.273cm}}{\pgfqpoint{0.1cm}{1.258cm}}{\pgfqpoint{0.137cm}{1.258cm}}
\pgfpathcurveto{\pgfqpoint{0.173cm}{1.258cm}}{\pgfqpoint{0.207cm}{1.273cm}}{\pgfqpoint{0.233cm}{1.299cm}}
\pgfpathcurveto{\pgfqpoint{0.259cm}{1.324cm}}{\pgfqpoint{0.273cm}{1.359cm}}{\pgfqpoint{0.273cm}{1.395cm}}
\pgfusepath{fill}
\begin{pgfscope}
\pgfsetdash{}{0cm}
\pgfsetlinewidth{0.818mm}
\pgfsetmiterlimit{7.0}
\pgfpathmoveto{\pgfqpoint{0.682cm}{0.671cm}}
\pgfpathlineto{\pgfqpoint{0.679cm}{1.418cm}}
\pgfusepath{stroke}
\end{pgfscope}
\pgfpathmoveto{\pgfqpoint{0.815cm}{1.399cm}}
\pgfpathcurveto{\pgfqpoint{0.815cm}{1.435cm}}{\pgfqpoint{0.801cm}{1.47cm}}{\pgfqpoint{0.775cm}{1.496cm}}
\pgfpathcurveto{\pgfqpoint{0.75cm}{1.521cm}}{\pgfqpoint{0.715cm}{1.536cm}}{\pgfqpoint{0.679cm}{1.536cm}}
\pgfpathcurveto{\pgfqpoint{0.643cm}{1.536cm}}{\pgfqpoint{0.608cm}{1.521cm}}{\pgfqpoint{0.582cm}{1.496cm}}
\pgfpathcurveto{\pgfqpoint{0.557cm}{1.47cm}}{\pgfqpoint{0.542cm}{1.435cm}}{\pgfqpoint{0.542cm}{1.399cm}}
\pgfpathcurveto{\pgfqpoint{0.542cm}{1.363cm}}{\pgfqpoint{0.557cm}{1.328cm}}{\pgfqpoint{0.582cm}{1.302cm}}
\pgfpathcurveto{\pgfqpoint{0.608cm}{1.276cm}}{\pgfqpoint{0.643cm}{1.262cm}}{\pgfqpoint{0.679cm}{1.262cm}}
\pgfpathcurveto{\pgfqpoint{0.715cm}{1.262cm}}{\pgfqpoint{0.75cm}{1.276cm}}{\pgfqpoint{0.775cm}{1.302cm}}
\pgfpathcurveto{\pgfqpoint{0.801cm}{1.328cm}}{\pgfqpoint{0.815cm}{1.363cm}}{\pgfqpoint{0.815cm}{1.399cm}}
\pgfusepath{fill}
\pgfpathmoveto{\pgfqpoint{1.345cm}{1.371cm}}
\pgfpathcurveto{\pgfqpoint{1.345cm}{1.408cm}}{\pgfqpoint{1.331cm}{1.442cm}}{\pgfqpoint{1.305cm}{1.468cm}}
\pgfpathcurveto{\pgfqpoint{1.28cm}{1.494cm}}{\pgfqpoint{1.245cm}{1.508cm}}{\pgfqpoint{1.209cm}{1.508cm}}
\pgfpathcurveto{\pgfqpoint{1.172cm}{1.508cm}}{\pgfqpoint{1.138cm}{1.494cm}}{\pgfqpoint{1.112cm}{1.468cm}}
\pgfpathcurveto{\pgfqpoint{1.087cm}{1.442cm}}{\pgfqpoint{1.072cm}{1.408cm}}{\pgfqpoint{1.072cm}{1.371cm}}
\pgfpathcurveto{\pgfqpoint{1.072cm}{1.335cm}}{\pgfqpoint{1.087cm}{1.3cm}}{\pgfqpoint{1.112cm}{1.274cm}}
\pgfpathcurveto{\pgfqpoint{1.138cm}{1.249cm}}{\pgfqpoint{1.172cm}{1.234cm}}{\pgfqpoint{1.209cm}{1.234cm}}
\pgfpathcurveto{\pgfqpoint{1.245cm}{1.234cm}}{\pgfqpoint{1.28cm}{1.249cm}}{\pgfqpoint{1.305cm}{1.274cm}}
\pgfpathcurveto{\pgfqpoint{1.331cm}{1.3cm}}{\pgfqpoint{1.345cm}{1.335cm}}{\pgfqpoint{1.345cm}{1.371cm}}
\pgfusepath{fill}
\begin{pgfscope}
\pgfsetdash{}{0cm}
\pgfsetlinewidth{0.818mm}
\pgfsetroundcap
\pgfsetmiterlimit{4.0}
\pgfpathmoveto{\pgfqpoint{0.682cm}{0.671cm}}
\pgfpathlineto{\pgfqpoint{0.682cm}{0.042cm}}
\pgfusepath{stroke}
\end{pgfscope}
\end{pgfscope}
\end{pgfscope}
\end{pgfscope}
\end{tikzpicture}}} + \phi + \psi)\\
&\quad= - \rmm{3X^{\!\resizebox{!}{.8em}{
\begin{tikzpicture}
\pgfpathmoveto{\pgfqpoint{0cm}{-0.035cm}}
\pgfpathlineto{\pgfqpoint{1.976cm}{-0.035cm}}
\pgfpathlineto{\pgfqpoint{1.976cm}{1.94cm}}
\pgfpathlineto{\pgfqpoint{0cm}{1.94cm}}
\pgfpathclose
\pgfusepath{clip}
\begin{pgfscope}
\begin{pgfscope}
\pgfpathmoveto{\pgfqpoint{0cm}{-0.035cm}}
\pgfpathlineto{\pgfqpoint{1.976cm}{-0.035cm}}
\pgfpathlineto{\pgfqpoint{1.976cm}{1.94cm}}
\pgfpathlineto{\pgfqpoint{0cm}{1.94cm}}
\pgfpathclose
\pgfusepath{clip}
\begin{pgfscope}
\begin{pgfscope}
\pgfsetdash{}{0cm}
\pgfsetlinewidth{0.818mm}
\pgfsetroundcap
\pgfsetroundjoin
\pgfsetmiterlimit{7.0}
\definecolor{eps2pgf_color}{gray}{0}\pgfsetstrokecolor{eps2pgf_color}\pgfsetfillcolor{eps2pgf_color}
\pgfpathmoveto{\pgfqpoint{0.117cm}{1.815cm}}
\pgfpathlineto{\pgfqpoint{0.682cm}{1.065cm}}
\pgfpathlineto{\pgfqpoint{1.246cm}{1.815cm}}
\pgfusepath{stroke}
\end{pgfscope}
\definecolor{eps2pgf_color}{gray}{0}\pgfsetstrokecolor{eps2pgf_color}\pgfsetfillcolor{eps2pgf_color}
\pgfpathmoveto{\pgfqpoint{0.273cm}{1.789cm}}
\pgfpathcurveto{\pgfqpoint{0.273cm}{1.825cm}}{\pgfqpoint{0.259cm}{1.86cm}}{\pgfqpoint{0.233cm}{1.886cm}}
\pgfpathcurveto{\pgfqpoint{0.207cm}{1.912cm}}{\pgfqpoint{0.173cm}{1.926cm}}{\pgfqpoint{0.137cm}{1.926cm}}
\pgfpathcurveto{\pgfqpoint{0.1cm}{1.926cm}}{\pgfqpoint{0.066cm}{1.912cm}}{\pgfqpoint{0.04cm}{1.886cm}}
\pgfpathcurveto{\pgfqpoint{0.014cm}{1.86cm}}{\pgfqpoint{0cm}{1.825cm}}{\pgfqpoint{0cm}{1.789cm}}
\pgfpathcurveto{\pgfqpoint{0cm}{1.753cm}}{\pgfqpoint{0.014cm}{1.718cm}}{\pgfqpoint{0.04cm}{1.692cm}}
\pgfpathcurveto{\pgfqpoint{0.066cm}{1.667cm}}{\pgfqpoint{0.1cm}{1.652cm}}{\pgfqpoint{0.137cm}{1.652cm}}
\pgfpathcurveto{\pgfqpoint{0.173cm}{1.652cm}}{\pgfqpoint{0.207cm}{1.667cm}}{\pgfqpoint{0.233cm}{1.692cm}}
\pgfpathcurveto{\pgfqpoint{0.259cm}{1.718cm}}{\pgfqpoint{0.273cm}{1.753cm}}{\pgfqpoint{0.273cm}{1.789cm}}
\pgfusepath{fill}
\begin{pgfscope}
\pgfsetdash{}{0cm}
\pgfsetlinewidth{0.818mm}
\pgfsetmiterlimit{7.0}
\pgfpathmoveto{\pgfqpoint{0.682cm}{1.065cm}}
\pgfpathlineto{\pgfqpoint{0.679cm}{1.812cm}}
\pgfusepath{stroke}
\end{pgfscope}
\pgfpathmoveto{\pgfqpoint{0.815cm}{1.793cm}}
\pgfpathcurveto{\pgfqpoint{0.815cm}{1.829cm}}{\pgfqpoint{0.801cm}{1.864cm}}{\pgfqpoint{0.775cm}{1.89cm}}
\pgfpathcurveto{\pgfqpoint{0.75cm}{1.915cm}}{\pgfqpoint{0.715cm}{1.93cm}}{\pgfqpoint{0.679cm}{1.93cm}}
\pgfpathcurveto{\pgfqpoint{0.643cm}{1.93cm}}{\pgfqpoint{0.608cm}{1.915cm}}{\pgfqpoint{0.582cm}{1.89cm}}
\pgfpathcurveto{\pgfqpoint{0.557cm}{1.864cm}}{\pgfqpoint{0.542cm}{1.829cm}}{\pgfqpoint{0.542cm}{1.793cm}}
\pgfpathcurveto{\pgfqpoint{0.542cm}{1.756cm}}{\pgfqpoint{0.557cm}{1.722cm}}{\pgfqpoint{0.582cm}{1.696cm}}
\pgfpathcurveto{\pgfqpoint{0.608cm}{1.67cm}}{\pgfqpoint{0.643cm}{1.656cm}}{\pgfqpoint{0.679cm}{1.656cm}}
\pgfpathcurveto{\pgfqpoint{0.715cm}{1.656cm}}{\pgfqpoint{0.75cm}{1.67cm}}{\pgfqpoint{0.775cm}{1.696cm}}
\pgfpathcurveto{\pgfqpoint{0.801cm}{1.722cm}}{\pgfqpoint{0.815cm}{1.756cm}}{\pgfqpoint{0.815cm}{1.793cm}}
\pgfusepath{fill}
\pgfpathmoveto{\pgfqpoint{1.345cm}{1.765cm}}
\pgfpathcurveto{\pgfqpoint{1.345cm}{1.801cm}}{\pgfqpoint{1.331cm}{1.836cm}}{\pgfqpoint{1.305cm}{1.862cm}}
\pgfpathcurveto{\pgfqpoint{1.28cm}{1.887cm}}{\pgfqpoint{1.245cm}{1.902cm}}{\pgfqpoint{1.209cm}{1.902cm}}
\pgfpathcurveto{\pgfqpoint{1.172cm}{1.902cm}}{\pgfqpoint{1.138cm}{1.887cm}}{\pgfqpoint{1.112cm}{1.862cm}}
\pgfpathcurveto{\pgfqpoint{1.087cm}{1.836cm}}{\pgfqpoint{1.072cm}{1.801cm}}{\pgfqpoint{1.072cm}{1.765cm}}
\pgfpathcurveto{\pgfqpoint{1.072cm}{1.728cm}}{\pgfqpoint{1.087cm}{1.694cm}}{\pgfqpoint{1.112cm}{1.668cm}}
\pgfpathcurveto{\pgfqpoint{1.138cm}{1.642cm}}{\pgfqpoint{1.172cm}{1.628cm}}{\pgfqpoint{1.209cm}{1.628cm}}
\pgfpathcurveto{\pgfqpoint{1.245cm}{1.628cm}}{\pgfqpoint{1.28cm}{1.642cm}}{\pgfqpoint{1.305cm}{1.668cm}}
\pgfpathcurveto{\pgfqpoint{1.331cm}{1.694cm}}{\pgfqpoint{1.345cm}{1.728cm}}{\pgfqpoint{1.345cm}{1.765cm}}
\pgfusepath{fill}
\begin{pgfscope}
\pgfsetdash{}{0cm}
\pgfsetlinewidth{0.818mm}
\pgfsetroundcap
\pgfsetroundjoin
\pgfsetmiterlimit{7.0}
\pgfpathmoveto{\pgfqpoint{0.682cm}{1.065cm}}
\pgfpathlineto{\pgfqpoint{1.246cm}{0.315cm}}
\pgfpathlineto{\pgfqpoint{1.811cm}{1.065cm}}
\pgfusepath{stroke}
\end{pgfscope}
\pgfpathmoveto{\pgfqpoint{1.948cm}{1.065cm}}
\pgfpathcurveto{\pgfqpoint{1.948cm}{1.101cm}}{\pgfqpoint{1.933cm}{1.136cm}}{\pgfqpoint{1.907cm}{1.162cm}}
\pgfpathcurveto{\pgfqpoint{1.882cm}{1.187cm}}{\pgfqpoint{1.847cm}{1.202cm}}{\pgfqpoint{1.811cm}{1.202cm}}
\pgfpathcurveto{\pgfqpoint{1.775cm}{1.202cm}}{\pgfqpoint{1.74cm}{1.187cm}}{\pgfqpoint{1.714cm}{1.162cm}}
\pgfpathcurveto{\pgfqpoint{1.689cm}{1.136cm}}{\pgfqpoint{1.674cm}{1.101cm}}{\pgfqpoint{1.674cm}{1.065cm}}
\pgfpathcurveto{\pgfqpoint{1.674cm}{1.029cm}}{\pgfqpoint{1.689cm}{0.994cm}}{\pgfqpoint{1.714cm}{0.968cm}}
\pgfpathcurveto{\pgfqpoint{1.74cm}{0.942cm}}{\pgfqpoint{1.775cm}{0.928cm}}{\pgfqpoint{1.811cm}{0.928cm}}
\pgfpathcurveto{\pgfqpoint{1.847cm}{0.928cm}}{\pgfqpoint{1.882cm}{0.942cm}}{\pgfqpoint{1.907cm}{0.968cm}}
\pgfpathcurveto{\pgfqpoint{1.933cm}{0.994cm}}{\pgfqpoint{1.948cm}{1.029cm}}{\pgfqpoint{1.948cm}{1.065cm}}
\pgfusepath{fill}
\begin{pgfscope}
\pgfsetdash{}{0cm}
\pgfsetlinewidth{0.818mm}
\pgfsetmiterlimit{7.0}
\pgfpathmoveto{\pgfqpoint{1.246cm}{0.315cm}}
\pgfpathlineto{\pgfqpoint{1.244cm}{1.061cm}}
\pgfusepath{stroke}
\end{pgfscope}
\pgfpathmoveto{\pgfqpoint{1.38cm}{1.065cm}}
\pgfpathcurveto{\pgfqpoint{1.38cm}{1.101cm}}{\pgfqpoint{1.366cm}{1.136cm}}{\pgfqpoint{1.34cm}{1.162cm}}
\pgfpathcurveto{\pgfqpoint{1.315cm}{1.187cm}}{\pgfqpoint{1.28cm}{1.202cm}}{\pgfqpoint{1.244cm}{1.202cm}}
\pgfpathcurveto{\pgfqpoint{1.207cm}{1.202cm}}{\pgfqpoint{1.173cm}{1.187cm}}{\pgfqpoint{1.147cm}{1.162cm}}
\pgfpathcurveto{\pgfqpoint{1.121cm}{1.136cm}}{\pgfqpoint{1.107cm}{1.101cm}}{\pgfqpoint{1.107cm}{1.065cm}}
\pgfpathcurveto{\pgfqpoint{1.107cm}{1.029cm}}{\pgfqpoint{1.121cm}{0.994cm}}{\pgfqpoint{1.147cm}{0.968cm}}
\pgfpathcurveto{\pgfqpoint{1.173cm}{0.942cm}}{\pgfqpoint{1.207cm}{0.928cm}}{\pgfqpoint{1.244cm}{0.928cm}}
\pgfpathcurveto{\pgfqpoint{1.28cm}{0.928cm}}{\pgfqpoint{1.315cm}{0.942cm}}{\pgfqpoint{1.34cm}{0.968cm}}
\pgfpathcurveto{\pgfqpoint{1.366cm}{0.994cm}}{\pgfqpoint{1.38cm}{1.029cm}}{\pgfqpoint{1.38cm}{1.065cm}}
\pgfusepath{fill}
\begin{pgfscope}
\pgfsetdash{}{0cm}
\pgfsetlinewidth{0.818mm}
\pgfsetmiterlimit{4.0}
\pgfpathmoveto{\pgfqpoint{1.383cm}{0.178cm}}
\pgfpathcurveto{\pgfqpoint{1.383cm}{0.214cm}}{\pgfqpoint{1.369cm}{0.249cm}}{\pgfqpoint{1.343cm}{0.275cm}}
\pgfpathcurveto{\pgfqpoint{1.317cm}{0.3cm}}{\pgfqpoint{1.283cm}{0.315cm}}{\pgfqpoint{1.246cm}{0.315cm}}
\pgfpathcurveto{\pgfqpoint{1.21cm}{0.315cm}}{\pgfqpoint{1.175cm}{0.3cm}}{\pgfqpoint{1.15cm}{0.275cm}}
\pgfpathcurveto{\pgfqpoint{1.124cm}{0.249cm}}{\pgfqpoint{1.11cm}{0.214cm}}{\pgfqpoint{1.11cm}{0.178cm}}
\pgfpathcurveto{\pgfqpoint{1.11cm}{0.141cm}}{\pgfqpoint{1.124cm}{0.107cm}}{\pgfqpoint{1.15cm}{0.081cm}}
\pgfpathcurveto{\pgfqpoint{1.175cm}{0.055cm}}{\pgfqpoint{1.21cm}{0.041cm}}{\pgfqpoint{1.246cm}{0.041cm}}
\pgfpathcurveto{\pgfqpoint{1.283cm}{0.041cm}}{\pgfqpoint{1.317cm}{0.055cm}}{\pgfqpoint{1.343cm}{0.081cm}}
\pgfpathcurveto{\pgfqpoint{1.369cm}{0.107cm}}{\pgfqpoint{1.383cm}{0.141cm}}{\pgfqpoint{1.383cm}{0.178cm}}
\pgfusepath{stroke}
\end{pgfscope}
\end{pgfscope}
\end{pgfscope}
\end{pgfscope}
\end{tikzpicture}}}}+\rmb{3\llbracket X^2 \rrbracket\circ\vartheta}-\rmg{3( - X^{\!\resizebox{0.6em}{!}{
\begin{tikzpicture}
\pgfpathmoveto{\pgfqpoint{0cm}{-0.035cm}}
\pgfpathlineto{\pgfqpoint{1.376cm}{-0.035cm}}
\pgfpathlineto{\pgfqpoint{1.376cm}{1.552cm}}
\pgfpathlineto{\pgfqpoint{0cm}{1.552cm}}
\pgfpathclose
\pgfusepath{clip}
\begin{pgfscope}
\begin{pgfscope}
\pgfpathmoveto{\pgfqpoint{0cm}{-0.035cm}}
\pgfpathlineto{\pgfqpoint{1.376cm}{-0.035cm}}
\pgfpathlineto{\pgfqpoint{1.376cm}{1.552cm}}
\pgfpathlineto{\pgfqpoint{0cm}{1.552cm}}
\pgfpathclose
\pgfusepath{clip}
\begin{pgfscope}
\begin{pgfscope}
\pgfsetdash{}{0cm}
\pgfsetlinewidth{0.818mm}
\pgfsetroundcap
\pgfsetroundjoin
\pgfsetmiterlimit{7.0}
\definecolor{eps2pgf_color}{gray}{0}\pgfsetstrokecolor{eps2pgf_color}\pgfsetfillcolor{eps2pgf_color}
\pgfpathmoveto{\pgfqpoint{0.117cm}{1.421cm}}
\pgfpathlineto{\pgfqpoint{0.682cm}{0.671cm}}
\pgfpathlineto{\pgfqpoint{1.246cm}{1.421cm}}
\pgfusepath{stroke}
\end{pgfscope}
\definecolor{eps2pgf_color}{gray}{0}\pgfsetstrokecolor{eps2pgf_color}\pgfsetfillcolor{eps2pgf_color}
\pgfpathmoveto{\pgfqpoint{0.273cm}{1.395cm}}
\pgfpathcurveto{\pgfqpoint{0.273cm}{1.432cm}}{\pgfqpoint{0.259cm}{1.467cm}}{\pgfqpoint{0.233cm}{1.492cm}}
\pgfpathcurveto{\pgfqpoint{0.207cm}{1.518cm}}{\pgfqpoint{0.173cm}{1.532cm}}{\pgfqpoint{0.137cm}{1.532cm}}
\pgfpathcurveto{\pgfqpoint{0.1cm}{1.532cm}}{\pgfqpoint{0.066cm}{1.518cm}}{\pgfqpoint{0.04cm}{1.492cm}}
\pgfpathcurveto{\pgfqpoint{0.014cm}{1.467cm}}{\pgfqpoint{0cm}{1.432cm}}{\pgfqpoint{0cm}{1.395cm}}
\pgfpathcurveto{\pgfqpoint{0cm}{1.359cm}}{\pgfqpoint{0.014cm}{1.324cm}}{\pgfqpoint{0.04cm}{1.299cm}}
\pgfpathcurveto{\pgfqpoint{0.066cm}{1.273cm}}{\pgfqpoint{0.1cm}{1.258cm}}{\pgfqpoint{0.137cm}{1.258cm}}
\pgfpathcurveto{\pgfqpoint{0.173cm}{1.258cm}}{\pgfqpoint{0.207cm}{1.273cm}}{\pgfqpoint{0.233cm}{1.299cm}}
\pgfpathcurveto{\pgfqpoint{0.259cm}{1.324cm}}{\pgfqpoint{0.273cm}{1.359cm}}{\pgfqpoint{0.273cm}{1.395cm}}
\pgfusepath{fill}
\begin{pgfscope}
\pgfsetdash{}{0cm}
\pgfsetlinewidth{0.818mm}
\pgfsetmiterlimit{7.0}
\pgfpathmoveto{\pgfqpoint{0.682cm}{0.671cm}}
\pgfpathlineto{\pgfqpoint{0.679cm}{1.418cm}}
\pgfusepath{stroke}
\end{pgfscope}
\pgfpathmoveto{\pgfqpoint{0.815cm}{1.399cm}}
\pgfpathcurveto{\pgfqpoint{0.815cm}{1.435cm}}{\pgfqpoint{0.801cm}{1.47cm}}{\pgfqpoint{0.775cm}{1.496cm}}
\pgfpathcurveto{\pgfqpoint{0.75cm}{1.521cm}}{\pgfqpoint{0.715cm}{1.536cm}}{\pgfqpoint{0.679cm}{1.536cm}}
\pgfpathcurveto{\pgfqpoint{0.643cm}{1.536cm}}{\pgfqpoint{0.608cm}{1.521cm}}{\pgfqpoint{0.582cm}{1.496cm}}
\pgfpathcurveto{\pgfqpoint{0.557cm}{1.47cm}}{\pgfqpoint{0.542cm}{1.435cm}}{\pgfqpoint{0.542cm}{1.399cm}}
\pgfpathcurveto{\pgfqpoint{0.542cm}{1.363cm}}{\pgfqpoint{0.557cm}{1.328cm}}{\pgfqpoint{0.582cm}{1.302cm}}
\pgfpathcurveto{\pgfqpoint{0.608cm}{1.276cm}}{\pgfqpoint{0.643cm}{1.262cm}}{\pgfqpoint{0.679cm}{1.262cm}}
\pgfpathcurveto{\pgfqpoint{0.715cm}{1.262cm}}{\pgfqpoint{0.75cm}{1.276cm}}{\pgfqpoint{0.775cm}{1.302cm}}
\pgfpathcurveto{\pgfqpoint{0.801cm}{1.328cm}}{\pgfqpoint{0.815cm}{1.363cm}}{\pgfqpoint{0.815cm}{1.399cm}}
\pgfusepath{fill}
\pgfpathmoveto{\pgfqpoint{1.345cm}{1.371cm}}
\pgfpathcurveto{\pgfqpoint{1.345cm}{1.408cm}}{\pgfqpoint{1.331cm}{1.442cm}}{\pgfqpoint{1.305cm}{1.468cm}}
\pgfpathcurveto{\pgfqpoint{1.28cm}{1.494cm}}{\pgfqpoint{1.245cm}{1.508cm}}{\pgfqpoint{1.209cm}{1.508cm}}
\pgfpathcurveto{\pgfqpoint{1.172cm}{1.508cm}}{\pgfqpoint{1.138cm}{1.494cm}}{\pgfqpoint{1.112cm}{1.468cm}}
\pgfpathcurveto{\pgfqpoint{1.087cm}{1.442cm}}{\pgfqpoint{1.072cm}{1.408cm}}{\pgfqpoint{1.072cm}{1.371cm}}
\pgfpathcurveto{\pgfqpoint{1.072cm}{1.335cm}}{\pgfqpoint{1.087cm}{1.3cm}}{\pgfqpoint{1.112cm}{1.274cm}}
\pgfpathcurveto{\pgfqpoint{1.138cm}{1.249cm}}{\pgfqpoint{1.172cm}{1.234cm}}{\pgfqpoint{1.209cm}{1.234cm}}
\pgfpathcurveto{\pgfqpoint{1.245cm}{1.234cm}}{\pgfqpoint{1.28cm}{1.249cm}}{\pgfqpoint{1.305cm}{1.274cm}}
\pgfpathcurveto{\pgfqpoint{1.331cm}{1.3cm}}{\pgfqpoint{1.345cm}{1.335cm}}{\pgfqpoint{1.345cm}{1.371cm}}
\pgfusepath{fill}
\begin{pgfscope}
\pgfsetdash{}{0cm}
\pgfsetlinewidth{0.818mm}
\pgfsetroundcap
\pgfsetmiterlimit{4.0}
\pgfpathmoveto{\pgfqpoint{0.682cm}{0.671cm}}
\pgfpathlineto{\pgfqpoint{0.682cm}{0.042cm}}
\pgfusepath{stroke}
\end{pgfscope}
\end{pgfscope}
\end{pgfscope}
\end{pgfscope}
\end{tikzpicture}}} + \phi + \psi)X^{\!\resizebox{!}{.8em}{
\begin{tikzpicture}
\pgfpathmoveto{\pgfqpoint{0cm}{-0.035cm}}
\pgfpathlineto{\pgfqpoint{1.976cm}{-0.035cm}}
\pgfpathlineto{\pgfqpoint{1.976cm}{1.94cm}}
\pgfpathlineto{\pgfqpoint{0cm}{1.94cm}}
\pgfpathclose
\pgfusepath{clip}
\begin{pgfscope}
\begin{pgfscope}
\pgfpathmoveto{\pgfqpoint{0cm}{-0.035cm}}
\pgfpathlineto{\pgfqpoint{1.976cm}{-0.035cm}}
\pgfpathlineto{\pgfqpoint{1.976cm}{1.94cm}}
\pgfpathlineto{\pgfqpoint{0cm}{1.94cm}}
\pgfpathclose
\pgfusepath{clip}
\begin{pgfscope}
\begin{pgfscope}
\pgfsetdash{}{0cm}
\pgfsetlinewidth{0.818mm}
\pgfsetroundcap
\pgfsetroundjoin
\pgfsetmiterlimit{7.0}
\definecolor{eps2pgf_color}{gray}{0}\pgfsetstrokecolor{eps2pgf_color}\pgfsetfillcolor{eps2pgf_color}
\pgfpathmoveto{\pgfqpoint{0.117cm}{1.815cm}}
\pgfpathlineto{\pgfqpoint{0.682cm}{1.065cm}}
\pgfpathlineto{\pgfqpoint{1.246cm}{1.815cm}}
\pgfusepath{stroke}
\end{pgfscope}
\definecolor{eps2pgf_color}{gray}{0}\pgfsetstrokecolor{eps2pgf_color}\pgfsetfillcolor{eps2pgf_color}
\pgfpathmoveto{\pgfqpoint{0.273cm}{1.789cm}}
\pgfpathcurveto{\pgfqpoint{0.273cm}{1.825cm}}{\pgfqpoint{0.259cm}{1.86cm}}{\pgfqpoint{0.233cm}{1.886cm}}
\pgfpathcurveto{\pgfqpoint{0.207cm}{1.912cm}}{\pgfqpoint{0.173cm}{1.926cm}}{\pgfqpoint{0.137cm}{1.926cm}}
\pgfpathcurveto{\pgfqpoint{0.1cm}{1.926cm}}{\pgfqpoint{0.066cm}{1.912cm}}{\pgfqpoint{0.04cm}{1.886cm}}
\pgfpathcurveto{\pgfqpoint{0.014cm}{1.86cm}}{\pgfqpoint{0cm}{1.825cm}}{\pgfqpoint{0cm}{1.789cm}}
\pgfpathcurveto{\pgfqpoint{0cm}{1.753cm}}{\pgfqpoint{0.014cm}{1.718cm}}{\pgfqpoint{0.04cm}{1.692cm}}
\pgfpathcurveto{\pgfqpoint{0.066cm}{1.667cm}}{\pgfqpoint{0.1cm}{1.652cm}}{\pgfqpoint{0.137cm}{1.652cm}}
\pgfpathcurveto{\pgfqpoint{0.173cm}{1.652cm}}{\pgfqpoint{0.207cm}{1.667cm}}{\pgfqpoint{0.233cm}{1.692cm}}
\pgfpathcurveto{\pgfqpoint{0.259cm}{1.718cm}}{\pgfqpoint{0.273cm}{1.753cm}}{\pgfqpoint{0.273cm}{1.789cm}}
\pgfusepath{fill}
\pgfpathmoveto{\pgfqpoint{1.345cm}{1.765cm}}
\pgfpathcurveto{\pgfqpoint{1.345cm}{1.801cm}}{\pgfqpoint{1.331cm}{1.836cm}}{\pgfqpoint{1.305cm}{1.862cm}}
\pgfpathcurveto{\pgfqpoint{1.28cm}{1.887cm}}{\pgfqpoint{1.245cm}{1.902cm}}{\pgfqpoint{1.209cm}{1.902cm}}
\pgfpathcurveto{\pgfqpoint{1.172cm}{1.902cm}}{\pgfqpoint{1.138cm}{1.887cm}}{\pgfqpoint{1.112cm}{1.862cm}}
\pgfpathcurveto{\pgfqpoint{1.087cm}{1.836cm}}{\pgfqpoint{1.072cm}{1.801cm}}{\pgfqpoint{1.072cm}{1.765cm}}
\pgfpathcurveto{\pgfqpoint{1.072cm}{1.728cm}}{\pgfqpoint{1.087cm}{1.694cm}}{\pgfqpoint{1.112cm}{1.668cm}}
\pgfpathcurveto{\pgfqpoint{1.138cm}{1.642cm}}{\pgfqpoint{1.172cm}{1.628cm}}{\pgfqpoint{1.209cm}{1.628cm}}
\pgfpathcurveto{\pgfqpoint{1.245cm}{1.628cm}}{\pgfqpoint{1.28cm}{1.642cm}}{\pgfqpoint{1.305cm}{1.668cm}}
\pgfpathcurveto{\pgfqpoint{1.331cm}{1.694cm}}{\pgfqpoint{1.345cm}{1.728cm}}{\pgfqpoint{1.345cm}{1.765cm}}
\pgfusepath{fill}
\begin{pgfscope}
\pgfsetdash{}{0cm}
\pgfsetlinewidth{0.818mm}
\pgfsetroundcap
\pgfsetroundjoin
\pgfsetmiterlimit{7.0}
\pgfpathmoveto{\pgfqpoint{0.682cm}{1.065cm}}
\pgfpathlineto{\pgfqpoint{1.246cm}{0.315cm}}
\pgfpathlineto{\pgfqpoint{1.811cm}{1.065cm}}
\pgfusepath{stroke}
\end{pgfscope}
\pgfpathmoveto{\pgfqpoint{1.948cm}{1.065cm}}
\pgfpathcurveto{\pgfqpoint{1.948cm}{1.101cm}}{\pgfqpoint{1.933cm}{1.136cm}}{\pgfqpoint{1.907cm}{1.162cm}}
\pgfpathcurveto{\pgfqpoint{1.882cm}{1.187cm}}{\pgfqpoint{1.847cm}{1.202cm}}{\pgfqpoint{1.811cm}{1.202cm}}
\pgfpathcurveto{\pgfqpoint{1.775cm}{1.202cm}}{\pgfqpoint{1.74cm}{1.187cm}}{\pgfqpoint{1.714cm}{1.162cm}}
\pgfpathcurveto{\pgfqpoint{1.689cm}{1.136cm}}{\pgfqpoint{1.674cm}{1.101cm}}{\pgfqpoint{1.674cm}{1.065cm}}
\pgfpathcurveto{\pgfqpoint{1.674cm}{1.029cm}}{\pgfqpoint{1.689cm}{0.994cm}}{\pgfqpoint{1.714cm}{0.968cm}}
\pgfpathcurveto{\pgfqpoint{1.74cm}{0.942cm}}{\pgfqpoint{1.775cm}{0.928cm}}{\pgfqpoint{1.811cm}{0.928cm}}
\pgfpathcurveto{\pgfqpoint{1.847cm}{0.928cm}}{\pgfqpoint{1.882cm}{0.942cm}}{\pgfqpoint{1.907cm}{0.968cm}}
\pgfpathcurveto{\pgfqpoint{1.933cm}{0.994cm}}{\pgfqpoint{1.948cm}{1.029cm}}{\pgfqpoint{1.948cm}{1.065cm}}
\pgfusepath{fill}
\begin{pgfscope}
\pgfsetdash{}{0cm}
\pgfsetlinewidth{0.818mm}
\pgfsetmiterlimit{7.0}
\pgfpathmoveto{\pgfqpoint{1.246cm}{0.315cm}}
\pgfpathlineto{\pgfqpoint{1.244cm}{1.061cm}}
\pgfusepath{stroke}
\end{pgfscope}
\pgfpathmoveto{\pgfqpoint{1.38cm}{1.065cm}}
\pgfpathcurveto{\pgfqpoint{1.38cm}{1.101cm}}{\pgfqpoint{1.366cm}{1.136cm}}{\pgfqpoint{1.34cm}{1.162cm}}
\pgfpathcurveto{\pgfqpoint{1.315cm}{1.187cm}}{\pgfqpoint{1.28cm}{1.202cm}}{\pgfqpoint{1.244cm}{1.202cm}}
\pgfpathcurveto{\pgfqpoint{1.207cm}{1.202cm}}{\pgfqpoint{1.173cm}{1.187cm}}{\pgfqpoint{1.147cm}{1.162cm}}
\pgfpathcurveto{\pgfqpoint{1.121cm}{1.136cm}}{\pgfqpoint{1.107cm}{1.101cm}}{\pgfqpoint{1.107cm}{1.065cm}}
\pgfpathcurveto{\pgfqpoint{1.107cm}{1.029cm}}{\pgfqpoint{1.121cm}{0.994cm}}{\pgfqpoint{1.147cm}{0.968cm}}
\pgfpathcurveto{\pgfqpoint{1.173cm}{0.942cm}}{\pgfqpoint{1.207cm}{0.928cm}}{\pgfqpoint{1.244cm}{0.928cm}}
\pgfpathcurveto{\pgfqpoint{1.28cm}{0.928cm}}{\pgfqpoint{1.315cm}{0.942cm}}{\pgfqpoint{1.34cm}{0.968cm}}
\pgfpathcurveto{\pgfqpoint{1.366cm}{0.994cm}}{\pgfqpoint{1.38cm}{1.029cm}}{\pgfqpoint{1.38cm}{1.065cm}}
\pgfusepath{fill}
\begin{pgfscope}
\pgfsetdash{}{0cm}
\pgfsetlinewidth{0.818mm}
\pgfsetmiterlimit{4.0}
\pgfpathmoveto{\pgfqpoint{1.383cm}{0.178cm}}
\pgfpathcurveto{\pgfqpoint{1.383cm}{0.214cm}}{\pgfqpoint{1.369cm}{0.249cm}}{\pgfqpoint{1.343cm}{0.275cm}}
\pgfpathcurveto{\pgfqpoint{1.317cm}{0.3cm}}{\pgfqpoint{1.283cm}{0.315cm}}{\pgfqpoint{1.246cm}{0.315cm}}
\pgfpathcurveto{\pgfqpoint{1.21cm}{0.315cm}}{\pgfqpoint{1.175cm}{0.3cm}}{\pgfqpoint{1.15cm}{0.275cm}}
\pgfpathcurveto{\pgfqpoint{1.124cm}{0.249cm}}{\pgfqpoint{1.11cm}{0.214cm}}{\pgfqpoint{1.11cm}{0.178cm}}
\pgfpathcurveto{\pgfqpoint{1.11cm}{0.141cm}}{\pgfqpoint{1.124cm}{0.107cm}}{\pgfqpoint{1.15cm}{0.081cm}}
\pgfpathcurveto{\pgfqpoint{1.175cm}{0.055cm}}{\pgfqpoint{1.21cm}{0.041cm}}{\pgfqpoint{1.246cm}{0.041cm}}
\pgfpathcurveto{\pgfqpoint{1.283cm}{0.041cm}}{\pgfqpoint{1.317cm}{0.055cm}}{\pgfqpoint{1.343cm}{0.081cm}}
\pgfpathcurveto{\pgfqpoint{1.369cm}{0.107cm}}{\pgfqpoint{1.383cm}{0.141cm}}{\pgfqpoint{1.383cm}{0.178cm}}
\pgfusepath{stroke}
\end{pgfscope}
\end{pgfscope}
\end{pgfscope}
\end{pgfscope}
\end{tikzpicture}}}}\\
&\qquad\qquad-\rmb{9\mathrm{com}(-X^{\!\resizebox{0.6em}{!}{
\begin{tikzpicture}
\pgfpathmoveto{\pgfqpoint{0cm}{-0.035cm}}
\pgfpathlineto{\pgfqpoint{1.376cm}{-0.035cm}}
\pgfpathlineto{\pgfqpoint{1.376cm}{1.552cm}}
\pgfpathlineto{\pgfqpoint{0cm}{1.552cm}}
\pgfpathclose
\pgfusepath{clip}
\begin{pgfscope}
\begin{pgfscope}
\pgfpathmoveto{\pgfqpoint{0cm}{-0.035cm}}
\pgfpathlineto{\pgfqpoint{1.376cm}{-0.035cm}}
\pgfpathlineto{\pgfqpoint{1.376cm}{1.552cm}}
\pgfpathlineto{\pgfqpoint{0cm}{1.552cm}}
\pgfpathclose
\pgfusepath{clip}
\begin{pgfscope}
\begin{pgfscope}
\pgfsetdash{}{0cm}
\pgfsetlinewidth{0.818mm}
\pgfsetroundcap
\pgfsetroundjoin
\pgfsetmiterlimit{7.0}
\definecolor{eps2pgf_color}{gray}{0}\pgfsetstrokecolor{eps2pgf_color}\pgfsetfillcolor{eps2pgf_color}
\pgfpathmoveto{\pgfqpoint{0.117cm}{1.421cm}}
\pgfpathlineto{\pgfqpoint{0.682cm}{0.671cm}}
\pgfpathlineto{\pgfqpoint{1.246cm}{1.421cm}}
\pgfusepath{stroke}
\end{pgfscope}
\definecolor{eps2pgf_color}{gray}{0}\pgfsetstrokecolor{eps2pgf_color}\pgfsetfillcolor{eps2pgf_color}
\pgfpathmoveto{\pgfqpoint{0.273cm}{1.395cm}}
\pgfpathcurveto{\pgfqpoint{0.273cm}{1.432cm}}{\pgfqpoint{0.259cm}{1.467cm}}{\pgfqpoint{0.233cm}{1.492cm}}
\pgfpathcurveto{\pgfqpoint{0.207cm}{1.518cm}}{\pgfqpoint{0.173cm}{1.532cm}}{\pgfqpoint{0.137cm}{1.532cm}}
\pgfpathcurveto{\pgfqpoint{0.1cm}{1.532cm}}{\pgfqpoint{0.066cm}{1.518cm}}{\pgfqpoint{0.04cm}{1.492cm}}
\pgfpathcurveto{\pgfqpoint{0.014cm}{1.467cm}}{\pgfqpoint{0cm}{1.432cm}}{\pgfqpoint{0cm}{1.395cm}}
\pgfpathcurveto{\pgfqpoint{0cm}{1.359cm}}{\pgfqpoint{0.014cm}{1.324cm}}{\pgfqpoint{0.04cm}{1.299cm}}
\pgfpathcurveto{\pgfqpoint{0.066cm}{1.273cm}}{\pgfqpoint{0.1cm}{1.258cm}}{\pgfqpoint{0.137cm}{1.258cm}}
\pgfpathcurveto{\pgfqpoint{0.173cm}{1.258cm}}{\pgfqpoint{0.207cm}{1.273cm}}{\pgfqpoint{0.233cm}{1.299cm}}
\pgfpathcurveto{\pgfqpoint{0.259cm}{1.324cm}}{\pgfqpoint{0.273cm}{1.359cm}}{\pgfqpoint{0.273cm}{1.395cm}}
\pgfusepath{fill}
\begin{pgfscope}
\pgfsetdash{}{0cm}
\pgfsetlinewidth{0.818mm}
\pgfsetmiterlimit{7.0}
\pgfpathmoveto{\pgfqpoint{0.682cm}{0.671cm}}
\pgfpathlineto{\pgfqpoint{0.679cm}{1.418cm}}
\pgfusepath{stroke}
\end{pgfscope}
\pgfpathmoveto{\pgfqpoint{0.815cm}{1.399cm}}
\pgfpathcurveto{\pgfqpoint{0.815cm}{1.435cm}}{\pgfqpoint{0.801cm}{1.47cm}}{\pgfqpoint{0.775cm}{1.496cm}}
\pgfpathcurveto{\pgfqpoint{0.75cm}{1.521cm}}{\pgfqpoint{0.715cm}{1.536cm}}{\pgfqpoint{0.679cm}{1.536cm}}
\pgfpathcurveto{\pgfqpoint{0.643cm}{1.536cm}}{\pgfqpoint{0.608cm}{1.521cm}}{\pgfqpoint{0.582cm}{1.496cm}}
\pgfpathcurveto{\pgfqpoint{0.557cm}{1.47cm}}{\pgfqpoint{0.542cm}{1.435cm}}{\pgfqpoint{0.542cm}{1.399cm}}
\pgfpathcurveto{\pgfqpoint{0.542cm}{1.363cm}}{\pgfqpoint{0.557cm}{1.328cm}}{\pgfqpoint{0.582cm}{1.302cm}}
\pgfpathcurveto{\pgfqpoint{0.608cm}{1.276cm}}{\pgfqpoint{0.643cm}{1.262cm}}{\pgfqpoint{0.679cm}{1.262cm}}
\pgfpathcurveto{\pgfqpoint{0.715cm}{1.262cm}}{\pgfqpoint{0.75cm}{1.276cm}}{\pgfqpoint{0.775cm}{1.302cm}}
\pgfpathcurveto{\pgfqpoint{0.801cm}{1.328cm}}{\pgfqpoint{0.815cm}{1.363cm}}{\pgfqpoint{0.815cm}{1.399cm}}
\pgfusepath{fill}
\pgfpathmoveto{\pgfqpoint{1.345cm}{1.371cm}}
\pgfpathcurveto{\pgfqpoint{1.345cm}{1.408cm}}{\pgfqpoint{1.331cm}{1.442cm}}{\pgfqpoint{1.305cm}{1.468cm}}
\pgfpathcurveto{\pgfqpoint{1.28cm}{1.494cm}}{\pgfqpoint{1.245cm}{1.508cm}}{\pgfqpoint{1.209cm}{1.508cm}}
\pgfpathcurveto{\pgfqpoint{1.172cm}{1.508cm}}{\pgfqpoint{1.138cm}{1.494cm}}{\pgfqpoint{1.112cm}{1.468cm}}
\pgfpathcurveto{\pgfqpoint{1.087cm}{1.442cm}}{\pgfqpoint{1.072cm}{1.408cm}}{\pgfqpoint{1.072cm}{1.371cm}}
\pgfpathcurveto{\pgfqpoint{1.072cm}{1.335cm}}{\pgfqpoint{1.087cm}{1.3cm}}{\pgfqpoint{1.112cm}{1.274cm}}
\pgfpathcurveto{\pgfqpoint{1.138cm}{1.249cm}}{\pgfqpoint{1.172cm}{1.234cm}}{\pgfqpoint{1.209cm}{1.234cm}}
\pgfpathcurveto{\pgfqpoint{1.245cm}{1.234cm}}{\pgfqpoint{1.28cm}{1.249cm}}{\pgfqpoint{1.305cm}{1.274cm}}
\pgfpathcurveto{\pgfqpoint{1.331cm}{1.3cm}}{\pgfqpoint{1.345cm}{1.335cm}}{\pgfqpoint{1.345cm}{1.371cm}}
\pgfusepath{fill}
\begin{pgfscope}
\pgfsetdash{}{0cm}
\pgfsetlinewidth{0.818mm}
\pgfsetroundcap
\pgfsetmiterlimit{4.0}
\pgfpathmoveto{\pgfqpoint{0.682cm}{0.671cm}}
\pgfpathlineto{\pgfqpoint{0.682cm}{0.042cm}}
\pgfusepath{stroke}
\end{pgfscope}
\end{pgfscope}
\end{pgfscope}
\end{pgfscope}
\end{tikzpicture}}} + \phi + \psi,X^{\!\resizebox{0.6em}{!}{
\begin{tikzpicture}
\pgfpathmoveto{\pgfqpoint{0cm}{0cm}}
\pgfpathlineto{\pgfqpoint{1.376cm}{0cm}}
\pgfpathlineto{\pgfqpoint{1.376cm}{1.588cm}}
\pgfpathlineto{\pgfqpoint{0cm}{1.588cm}}
\pgfpathclose
\pgfusepath{clip}
\begin{pgfscope}
\begin{pgfscope}
\pgfpathmoveto{\pgfqpoint{0cm}{0cm}}
\pgfpathlineto{\pgfqpoint{1.376cm}{0cm}}
\pgfpathlineto{\pgfqpoint{1.376cm}{1.588cm}}
\pgfpathlineto{\pgfqpoint{0cm}{1.588cm}}
\pgfpathclose
\pgfusepath{clip}
\begin{pgfscope}
\begin{pgfscope}
\definecolor{eps2pgf_color}{gray}{0.976471}\pgfsetstrokecolor{eps2pgf_color}\pgfsetfillcolor{eps2pgf_color}
\pgfpathmoveto{\pgfqpoint{0cm}{0cm}}
\pgfpathlineto{\pgfqpoint{1.376cm}{0cm}}
\pgfpathlineto{\pgfqpoint{1.376cm}{1.588cm}}
\pgfpathlineto{\pgfqpoint{0cm}{1.588cm}}
\pgfpathclose
\pgfusepath{fill}
\end{pgfscope}
\begin{pgfscope}
\pgfsetdash{}{0cm}
\pgfsetlinewidth{0.818mm}
\pgfsetroundcap
\pgfsetroundjoin
\pgfsetmiterlimit{7.0}
\definecolor{eps2pgf_color}{gray}{0}\pgfsetstrokecolor{eps2pgf_color}\pgfsetfillcolor{eps2pgf_color}
\pgfpathmoveto{\pgfqpoint{0.117cm}{1.476cm}}
\pgfpathlineto{\pgfqpoint{0.682cm}{0.726cm}}
\pgfpathlineto{\pgfqpoint{1.246cm}{1.476cm}}
\pgfusepath{stroke}
\end{pgfscope}
\definecolor{eps2pgf_color}{gray}{0}\pgfsetstrokecolor{eps2pgf_color}\pgfsetfillcolor{eps2pgf_color}
\pgfpathmoveto{\pgfqpoint{0.273cm}{1.451cm}}
\pgfpathcurveto{\pgfqpoint{0.273cm}{1.487cm}}{\pgfqpoint{0.259cm}{1.522cm}}{\pgfqpoint{0.233cm}{1.547cm}}
\pgfpathcurveto{\pgfqpoint{0.207cm}{1.573cm}}{\pgfqpoint{0.173cm}{1.588cm}}{\pgfqpoint{0.137cm}{1.588cm}}
\pgfpathcurveto{\pgfqpoint{0.1cm}{1.588cm}}{\pgfqpoint{0.066cm}{1.573cm}}{\pgfqpoint{0.04cm}{1.547cm}}
\pgfpathcurveto{\pgfqpoint{0.014cm}{1.522cm}}{\pgfqpoint{0cm}{1.487cm}}{\pgfqpoint{0cm}{1.451cm}}
\pgfpathcurveto{\pgfqpoint{0cm}{1.414cm}}{\pgfqpoint{0.014cm}{1.379cm}}{\pgfqpoint{0.04cm}{1.354cm}}
\pgfpathcurveto{\pgfqpoint{0.066cm}{1.328cm}}{\pgfqpoint{0.1cm}{1.314cm}}{\pgfqpoint{0.137cm}{1.314cm}}
\pgfpathcurveto{\pgfqpoint{0.173cm}{1.314cm}}{\pgfqpoint{0.207cm}{1.328cm}}{\pgfqpoint{0.233cm}{1.354cm}}
\pgfpathcurveto{\pgfqpoint{0.259cm}{1.379cm}}{\pgfqpoint{0.273cm}{1.414cm}}{\pgfqpoint{0.273cm}{1.451cm}}
\pgfusepath{fill}
\pgfpathmoveto{\pgfqpoint{1.345cm}{1.426cm}}
\pgfpathcurveto{\pgfqpoint{1.345cm}{1.463cm}}{\pgfqpoint{1.331cm}{1.497cm}}{\pgfqpoint{1.305cm}{1.523cm}}
\pgfpathcurveto{\pgfqpoint{1.28cm}{1.549cm}}{\pgfqpoint{1.245cm}{1.563cm}}{\pgfqpoint{1.209cm}{1.563cm}}
\pgfpathcurveto{\pgfqpoint{1.172cm}{1.563cm}}{\pgfqpoint{1.138cm}{1.549cm}}{\pgfqpoint{1.112cm}{1.523cm}}
\pgfpathcurveto{\pgfqpoint{1.087cm}{1.497cm}}{\pgfqpoint{1.072cm}{1.463cm}}{\pgfqpoint{1.072cm}{1.426cm}}
\pgfpathcurveto{\pgfqpoint{1.072cm}{1.39cm}}{\pgfqpoint{1.087cm}{1.355cm}}{\pgfqpoint{1.112cm}{1.329cm}}
\pgfpathcurveto{\pgfqpoint{1.138cm}{1.304cm}}{\pgfqpoint{1.172cm}{1.289cm}}{\pgfqpoint{1.209cm}{1.289cm}}
\pgfpathcurveto{\pgfqpoint{1.245cm}{1.289cm}}{\pgfqpoint{1.28cm}{1.304cm}}{\pgfqpoint{1.305cm}{1.329cm}}
\pgfpathcurveto{\pgfqpoint{1.331cm}{1.355cm}}{\pgfqpoint{1.345cm}{1.39cm}}{\pgfqpoint{1.345cm}{1.426cm}}
\pgfusepath{fill}
\begin{pgfscope}
\pgfsetdash{}{0cm}
\pgfsetlinewidth{0.818mm}
\pgfsetroundcap
\pgfsetmiterlimit{4.0}
\pgfpathmoveto{\pgfqpoint{0.682cm}{0.726cm}}
\pgfpathlineto{\pgfqpoint{0.682cm}{0.097cm}}
\pgfusepath{stroke}
\end{pgfscope}
\end{pgfscope}
\end{pgfscope}
\end{pgfscope}
\end{tikzpicture}}},\llbracket X^2 \rrbracket)}.
\end{align*}
Next, we have
\begin{align*}
&3X(-X^{\!\resizebox{0.6em}{!}{
\begin{tikzpicture}
\pgfpathmoveto{\pgfqpoint{0cm}{-0.035cm}}
\pgfpathlineto{\pgfqpoint{1.376cm}{-0.035cm}}
\pgfpathlineto{\pgfqpoint{1.376cm}{1.552cm}}
\pgfpathlineto{\pgfqpoint{0cm}{1.552cm}}
\pgfpathclose
\pgfusepath{clip}
\begin{pgfscope}
\begin{pgfscope}
\pgfpathmoveto{\pgfqpoint{0cm}{-0.035cm}}
\pgfpathlineto{\pgfqpoint{1.376cm}{-0.035cm}}
\pgfpathlineto{\pgfqpoint{1.376cm}{1.552cm}}
\pgfpathlineto{\pgfqpoint{0cm}{1.552cm}}
\pgfpathclose
\pgfusepath{clip}
\begin{pgfscope}
\begin{pgfscope}
\pgfsetdash{}{0cm}
\pgfsetlinewidth{0.818mm}
\pgfsetroundcap
\pgfsetroundjoin
\pgfsetmiterlimit{7.0}
\definecolor{eps2pgf_color}{gray}{0}\pgfsetstrokecolor{eps2pgf_color}\pgfsetfillcolor{eps2pgf_color}
\pgfpathmoveto{\pgfqpoint{0.117cm}{1.421cm}}
\pgfpathlineto{\pgfqpoint{0.682cm}{0.671cm}}
\pgfpathlineto{\pgfqpoint{1.246cm}{1.421cm}}
\pgfusepath{stroke}
\end{pgfscope}
\definecolor{eps2pgf_color}{gray}{0}\pgfsetstrokecolor{eps2pgf_color}\pgfsetfillcolor{eps2pgf_color}
\pgfpathmoveto{\pgfqpoint{0.273cm}{1.395cm}}
\pgfpathcurveto{\pgfqpoint{0.273cm}{1.432cm}}{\pgfqpoint{0.259cm}{1.467cm}}{\pgfqpoint{0.233cm}{1.492cm}}
\pgfpathcurveto{\pgfqpoint{0.207cm}{1.518cm}}{\pgfqpoint{0.173cm}{1.532cm}}{\pgfqpoint{0.137cm}{1.532cm}}
\pgfpathcurveto{\pgfqpoint{0.1cm}{1.532cm}}{\pgfqpoint{0.066cm}{1.518cm}}{\pgfqpoint{0.04cm}{1.492cm}}
\pgfpathcurveto{\pgfqpoint{0.014cm}{1.467cm}}{\pgfqpoint{0cm}{1.432cm}}{\pgfqpoint{0cm}{1.395cm}}
\pgfpathcurveto{\pgfqpoint{0cm}{1.359cm}}{\pgfqpoint{0.014cm}{1.324cm}}{\pgfqpoint{0.04cm}{1.299cm}}
\pgfpathcurveto{\pgfqpoint{0.066cm}{1.273cm}}{\pgfqpoint{0.1cm}{1.258cm}}{\pgfqpoint{0.137cm}{1.258cm}}
\pgfpathcurveto{\pgfqpoint{0.173cm}{1.258cm}}{\pgfqpoint{0.207cm}{1.273cm}}{\pgfqpoint{0.233cm}{1.299cm}}
\pgfpathcurveto{\pgfqpoint{0.259cm}{1.324cm}}{\pgfqpoint{0.273cm}{1.359cm}}{\pgfqpoint{0.273cm}{1.395cm}}
\pgfusepath{fill}
\begin{pgfscope}
\pgfsetdash{}{0cm}
\pgfsetlinewidth{0.818mm}
\pgfsetmiterlimit{7.0}
\pgfpathmoveto{\pgfqpoint{0.682cm}{0.671cm}}
\pgfpathlineto{\pgfqpoint{0.679cm}{1.418cm}}
\pgfusepath{stroke}
\end{pgfscope}
\pgfpathmoveto{\pgfqpoint{0.815cm}{1.399cm}}
\pgfpathcurveto{\pgfqpoint{0.815cm}{1.435cm}}{\pgfqpoint{0.801cm}{1.47cm}}{\pgfqpoint{0.775cm}{1.496cm}}
\pgfpathcurveto{\pgfqpoint{0.75cm}{1.521cm}}{\pgfqpoint{0.715cm}{1.536cm}}{\pgfqpoint{0.679cm}{1.536cm}}
\pgfpathcurveto{\pgfqpoint{0.643cm}{1.536cm}}{\pgfqpoint{0.608cm}{1.521cm}}{\pgfqpoint{0.582cm}{1.496cm}}
\pgfpathcurveto{\pgfqpoint{0.557cm}{1.47cm}}{\pgfqpoint{0.542cm}{1.435cm}}{\pgfqpoint{0.542cm}{1.399cm}}
\pgfpathcurveto{\pgfqpoint{0.542cm}{1.363cm}}{\pgfqpoint{0.557cm}{1.328cm}}{\pgfqpoint{0.582cm}{1.302cm}}
\pgfpathcurveto{\pgfqpoint{0.608cm}{1.276cm}}{\pgfqpoint{0.643cm}{1.262cm}}{\pgfqpoint{0.679cm}{1.262cm}}
\pgfpathcurveto{\pgfqpoint{0.715cm}{1.262cm}}{\pgfqpoint{0.75cm}{1.276cm}}{\pgfqpoint{0.775cm}{1.302cm}}
\pgfpathcurveto{\pgfqpoint{0.801cm}{1.328cm}}{\pgfqpoint{0.815cm}{1.363cm}}{\pgfqpoint{0.815cm}{1.399cm}}
\pgfusepath{fill}
\pgfpathmoveto{\pgfqpoint{1.345cm}{1.371cm}}
\pgfpathcurveto{\pgfqpoint{1.345cm}{1.408cm}}{\pgfqpoint{1.331cm}{1.442cm}}{\pgfqpoint{1.305cm}{1.468cm}}
\pgfpathcurveto{\pgfqpoint{1.28cm}{1.494cm}}{\pgfqpoint{1.245cm}{1.508cm}}{\pgfqpoint{1.209cm}{1.508cm}}
\pgfpathcurveto{\pgfqpoint{1.172cm}{1.508cm}}{\pgfqpoint{1.138cm}{1.494cm}}{\pgfqpoint{1.112cm}{1.468cm}}
\pgfpathcurveto{\pgfqpoint{1.087cm}{1.442cm}}{\pgfqpoint{1.072cm}{1.408cm}}{\pgfqpoint{1.072cm}{1.371cm}}
\pgfpathcurveto{\pgfqpoint{1.072cm}{1.335cm}}{\pgfqpoint{1.087cm}{1.3cm}}{\pgfqpoint{1.112cm}{1.274cm}}
\pgfpathcurveto{\pgfqpoint{1.138cm}{1.249cm}}{\pgfqpoint{1.172cm}{1.234cm}}{\pgfqpoint{1.209cm}{1.234cm}}
\pgfpathcurveto{\pgfqpoint{1.245cm}{1.234cm}}{\pgfqpoint{1.28cm}{1.249cm}}{\pgfqpoint{1.305cm}{1.274cm}}
\pgfpathcurveto{\pgfqpoint{1.331cm}{1.3cm}}{\pgfqpoint{1.345cm}{1.335cm}}{\pgfqpoint{1.345cm}{1.371cm}}
\pgfusepath{fill}
\begin{pgfscope}
\pgfsetdash{}{0cm}
\pgfsetlinewidth{0.818mm}
\pgfsetroundcap
\pgfsetmiterlimit{4.0}
\pgfpathmoveto{\pgfqpoint{0.682cm}{0.671cm}}
\pgfpathlineto{\pgfqpoint{0.682cm}{0.042cm}}
\pgfusepath{stroke}
\end{pgfscope}
\end{pgfscope}
\end{pgfscope}
\end{pgfscope}
\end{tikzpicture}}}+\phi+\psi)^2=3X(X^{\!\resizebox{0.6em}{!}{
\begin{tikzpicture}
\pgfpathmoveto{\pgfqpoint{0cm}{-0.035cm}}
\pgfpathlineto{\pgfqpoint{1.376cm}{-0.035cm}}
\pgfpathlineto{\pgfqpoint{1.376cm}{1.552cm}}
\pgfpathlineto{\pgfqpoint{0cm}{1.552cm}}
\pgfpathclose
\pgfusepath{clip}
\begin{pgfscope}
\begin{pgfscope}
\pgfpathmoveto{\pgfqpoint{0cm}{-0.035cm}}
\pgfpathlineto{\pgfqpoint{1.376cm}{-0.035cm}}
\pgfpathlineto{\pgfqpoint{1.376cm}{1.552cm}}
\pgfpathlineto{\pgfqpoint{0cm}{1.552cm}}
\pgfpathclose
\pgfusepath{clip}
\begin{pgfscope}
\begin{pgfscope}
\pgfsetdash{}{0cm}
\pgfsetlinewidth{0.818mm}
\pgfsetroundcap
\pgfsetroundjoin
\pgfsetmiterlimit{7.0}
\definecolor{eps2pgf_color}{gray}{0}\pgfsetstrokecolor{eps2pgf_color}\pgfsetfillcolor{eps2pgf_color}
\pgfpathmoveto{\pgfqpoint{0.117cm}{1.421cm}}
\pgfpathlineto{\pgfqpoint{0.682cm}{0.671cm}}
\pgfpathlineto{\pgfqpoint{1.246cm}{1.421cm}}
\pgfusepath{stroke}
\end{pgfscope}
\definecolor{eps2pgf_color}{gray}{0}\pgfsetstrokecolor{eps2pgf_color}\pgfsetfillcolor{eps2pgf_color}
\pgfpathmoveto{\pgfqpoint{0.273cm}{1.395cm}}
\pgfpathcurveto{\pgfqpoint{0.273cm}{1.432cm}}{\pgfqpoint{0.259cm}{1.467cm}}{\pgfqpoint{0.233cm}{1.492cm}}
\pgfpathcurveto{\pgfqpoint{0.207cm}{1.518cm}}{\pgfqpoint{0.173cm}{1.532cm}}{\pgfqpoint{0.137cm}{1.532cm}}
\pgfpathcurveto{\pgfqpoint{0.1cm}{1.532cm}}{\pgfqpoint{0.066cm}{1.518cm}}{\pgfqpoint{0.04cm}{1.492cm}}
\pgfpathcurveto{\pgfqpoint{0.014cm}{1.467cm}}{\pgfqpoint{0cm}{1.432cm}}{\pgfqpoint{0cm}{1.395cm}}
\pgfpathcurveto{\pgfqpoint{0cm}{1.359cm}}{\pgfqpoint{0.014cm}{1.324cm}}{\pgfqpoint{0.04cm}{1.299cm}}
\pgfpathcurveto{\pgfqpoint{0.066cm}{1.273cm}}{\pgfqpoint{0.1cm}{1.258cm}}{\pgfqpoint{0.137cm}{1.258cm}}
\pgfpathcurveto{\pgfqpoint{0.173cm}{1.258cm}}{\pgfqpoint{0.207cm}{1.273cm}}{\pgfqpoint{0.233cm}{1.299cm}}
\pgfpathcurveto{\pgfqpoint{0.259cm}{1.324cm}}{\pgfqpoint{0.273cm}{1.359cm}}{\pgfqpoint{0.273cm}{1.395cm}}
\pgfusepath{fill}
\begin{pgfscope}
\pgfsetdash{}{0cm}
\pgfsetlinewidth{0.818mm}
\pgfsetmiterlimit{7.0}
\pgfpathmoveto{\pgfqpoint{0.682cm}{0.671cm}}
\pgfpathlineto{\pgfqpoint{0.679cm}{1.418cm}}
\pgfusepath{stroke}
\end{pgfscope}
\pgfpathmoveto{\pgfqpoint{0.815cm}{1.399cm}}
\pgfpathcurveto{\pgfqpoint{0.815cm}{1.435cm}}{\pgfqpoint{0.801cm}{1.47cm}}{\pgfqpoint{0.775cm}{1.496cm}}
\pgfpathcurveto{\pgfqpoint{0.75cm}{1.521cm}}{\pgfqpoint{0.715cm}{1.536cm}}{\pgfqpoint{0.679cm}{1.536cm}}
\pgfpathcurveto{\pgfqpoint{0.643cm}{1.536cm}}{\pgfqpoint{0.608cm}{1.521cm}}{\pgfqpoint{0.582cm}{1.496cm}}
\pgfpathcurveto{\pgfqpoint{0.557cm}{1.47cm}}{\pgfqpoint{0.542cm}{1.435cm}}{\pgfqpoint{0.542cm}{1.399cm}}
\pgfpathcurveto{\pgfqpoint{0.542cm}{1.363cm}}{\pgfqpoint{0.557cm}{1.328cm}}{\pgfqpoint{0.582cm}{1.302cm}}
\pgfpathcurveto{\pgfqpoint{0.608cm}{1.276cm}}{\pgfqpoint{0.643cm}{1.262cm}}{\pgfqpoint{0.679cm}{1.262cm}}
\pgfpathcurveto{\pgfqpoint{0.715cm}{1.262cm}}{\pgfqpoint{0.75cm}{1.276cm}}{\pgfqpoint{0.775cm}{1.302cm}}
\pgfpathcurveto{\pgfqpoint{0.801cm}{1.328cm}}{\pgfqpoint{0.815cm}{1.363cm}}{\pgfqpoint{0.815cm}{1.399cm}}
\pgfusepath{fill}
\pgfpathmoveto{\pgfqpoint{1.345cm}{1.371cm}}
\pgfpathcurveto{\pgfqpoint{1.345cm}{1.408cm}}{\pgfqpoint{1.331cm}{1.442cm}}{\pgfqpoint{1.305cm}{1.468cm}}
\pgfpathcurveto{\pgfqpoint{1.28cm}{1.494cm}}{\pgfqpoint{1.245cm}{1.508cm}}{\pgfqpoint{1.209cm}{1.508cm}}
\pgfpathcurveto{\pgfqpoint{1.172cm}{1.508cm}}{\pgfqpoint{1.138cm}{1.494cm}}{\pgfqpoint{1.112cm}{1.468cm}}
\pgfpathcurveto{\pgfqpoint{1.087cm}{1.442cm}}{\pgfqpoint{1.072cm}{1.408cm}}{\pgfqpoint{1.072cm}{1.371cm}}
\pgfpathcurveto{\pgfqpoint{1.072cm}{1.335cm}}{\pgfqpoint{1.087cm}{1.3cm}}{\pgfqpoint{1.112cm}{1.274cm}}
\pgfpathcurveto{\pgfqpoint{1.138cm}{1.249cm}}{\pgfqpoint{1.172cm}{1.234cm}}{\pgfqpoint{1.209cm}{1.234cm}}
\pgfpathcurveto{\pgfqpoint{1.245cm}{1.234cm}}{\pgfqpoint{1.28cm}{1.249cm}}{\pgfqpoint{1.305cm}{1.274cm}}
\pgfpathcurveto{\pgfqpoint{1.331cm}{1.3cm}}{\pgfqpoint{1.345cm}{1.335cm}}{\pgfqpoint{1.345cm}{1.371cm}}
\pgfusepath{fill}
\begin{pgfscope}
\pgfsetdash{}{0cm}
\pgfsetlinewidth{0.818mm}
\pgfsetroundcap
\pgfsetmiterlimit{4.0}
\pgfpathmoveto{\pgfqpoint{0.682cm}{0.671cm}}
\pgfpathlineto{\pgfqpoint{0.682cm}{0.042cm}}
\pgfusepath{stroke}
\end{pgfscope}
\end{pgfscope}
\end{pgfscope}
\end{pgfscope}
\end{tikzpicture}}})^2-6XX^{\!\resizebox{0.6em}{!}{
\begin{tikzpicture}
\pgfpathmoveto{\pgfqpoint{0cm}{-0.035cm}}
\pgfpathlineto{\pgfqpoint{1.376cm}{-0.035cm}}
\pgfpathlineto{\pgfqpoint{1.376cm}{1.552cm}}
\pgfpathlineto{\pgfqpoint{0cm}{1.552cm}}
\pgfpathclose
\pgfusepath{clip}
\begin{pgfscope}
\begin{pgfscope}
\pgfpathmoveto{\pgfqpoint{0cm}{-0.035cm}}
\pgfpathlineto{\pgfqpoint{1.376cm}{-0.035cm}}
\pgfpathlineto{\pgfqpoint{1.376cm}{1.552cm}}
\pgfpathlineto{\pgfqpoint{0cm}{1.552cm}}
\pgfpathclose
\pgfusepath{clip}
\begin{pgfscope}
\begin{pgfscope}
\pgfsetdash{}{0cm}
\pgfsetlinewidth{0.818mm}
\pgfsetroundcap
\pgfsetroundjoin
\pgfsetmiterlimit{7.0}
\definecolor{eps2pgf_color}{gray}{0}\pgfsetstrokecolor{eps2pgf_color}\pgfsetfillcolor{eps2pgf_color}
\pgfpathmoveto{\pgfqpoint{0.117cm}{1.421cm}}
\pgfpathlineto{\pgfqpoint{0.682cm}{0.671cm}}
\pgfpathlineto{\pgfqpoint{1.246cm}{1.421cm}}
\pgfusepath{stroke}
\end{pgfscope}
\definecolor{eps2pgf_color}{gray}{0}\pgfsetstrokecolor{eps2pgf_color}\pgfsetfillcolor{eps2pgf_color}
\pgfpathmoveto{\pgfqpoint{0.273cm}{1.395cm}}
\pgfpathcurveto{\pgfqpoint{0.273cm}{1.432cm}}{\pgfqpoint{0.259cm}{1.467cm}}{\pgfqpoint{0.233cm}{1.492cm}}
\pgfpathcurveto{\pgfqpoint{0.207cm}{1.518cm}}{\pgfqpoint{0.173cm}{1.532cm}}{\pgfqpoint{0.137cm}{1.532cm}}
\pgfpathcurveto{\pgfqpoint{0.1cm}{1.532cm}}{\pgfqpoint{0.066cm}{1.518cm}}{\pgfqpoint{0.04cm}{1.492cm}}
\pgfpathcurveto{\pgfqpoint{0.014cm}{1.467cm}}{\pgfqpoint{0cm}{1.432cm}}{\pgfqpoint{0cm}{1.395cm}}
\pgfpathcurveto{\pgfqpoint{0cm}{1.359cm}}{\pgfqpoint{0.014cm}{1.324cm}}{\pgfqpoint{0.04cm}{1.299cm}}
\pgfpathcurveto{\pgfqpoint{0.066cm}{1.273cm}}{\pgfqpoint{0.1cm}{1.258cm}}{\pgfqpoint{0.137cm}{1.258cm}}
\pgfpathcurveto{\pgfqpoint{0.173cm}{1.258cm}}{\pgfqpoint{0.207cm}{1.273cm}}{\pgfqpoint{0.233cm}{1.299cm}}
\pgfpathcurveto{\pgfqpoint{0.259cm}{1.324cm}}{\pgfqpoint{0.273cm}{1.359cm}}{\pgfqpoint{0.273cm}{1.395cm}}
\pgfusepath{fill}
\begin{pgfscope}
\pgfsetdash{}{0cm}
\pgfsetlinewidth{0.818mm}
\pgfsetmiterlimit{7.0}
\pgfpathmoveto{\pgfqpoint{0.682cm}{0.671cm}}
\pgfpathlineto{\pgfqpoint{0.679cm}{1.418cm}}
\pgfusepath{stroke}
\end{pgfscope}
\pgfpathmoveto{\pgfqpoint{0.815cm}{1.399cm}}
\pgfpathcurveto{\pgfqpoint{0.815cm}{1.435cm}}{\pgfqpoint{0.801cm}{1.47cm}}{\pgfqpoint{0.775cm}{1.496cm}}
\pgfpathcurveto{\pgfqpoint{0.75cm}{1.521cm}}{\pgfqpoint{0.715cm}{1.536cm}}{\pgfqpoint{0.679cm}{1.536cm}}
\pgfpathcurveto{\pgfqpoint{0.643cm}{1.536cm}}{\pgfqpoint{0.608cm}{1.521cm}}{\pgfqpoint{0.582cm}{1.496cm}}
\pgfpathcurveto{\pgfqpoint{0.557cm}{1.47cm}}{\pgfqpoint{0.542cm}{1.435cm}}{\pgfqpoint{0.542cm}{1.399cm}}
\pgfpathcurveto{\pgfqpoint{0.542cm}{1.363cm}}{\pgfqpoint{0.557cm}{1.328cm}}{\pgfqpoint{0.582cm}{1.302cm}}
\pgfpathcurveto{\pgfqpoint{0.608cm}{1.276cm}}{\pgfqpoint{0.643cm}{1.262cm}}{\pgfqpoint{0.679cm}{1.262cm}}
\pgfpathcurveto{\pgfqpoint{0.715cm}{1.262cm}}{\pgfqpoint{0.75cm}{1.276cm}}{\pgfqpoint{0.775cm}{1.302cm}}
\pgfpathcurveto{\pgfqpoint{0.801cm}{1.328cm}}{\pgfqpoint{0.815cm}{1.363cm}}{\pgfqpoint{0.815cm}{1.399cm}}
\pgfusepath{fill}
\pgfpathmoveto{\pgfqpoint{1.345cm}{1.371cm}}
\pgfpathcurveto{\pgfqpoint{1.345cm}{1.408cm}}{\pgfqpoint{1.331cm}{1.442cm}}{\pgfqpoint{1.305cm}{1.468cm}}
\pgfpathcurveto{\pgfqpoint{1.28cm}{1.494cm}}{\pgfqpoint{1.245cm}{1.508cm}}{\pgfqpoint{1.209cm}{1.508cm}}
\pgfpathcurveto{\pgfqpoint{1.172cm}{1.508cm}}{\pgfqpoint{1.138cm}{1.494cm}}{\pgfqpoint{1.112cm}{1.468cm}}
\pgfpathcurveto{\pgfqpoint{1.087cm}{1.442cm}}{\pgfqpoint{1.072cm}{1.408cm}}{\pgfqpoint{1.072cm}{1.371cm}}
\pgfpathcurveto{\pgfqpoint{1.072cm}{1.335cm}}{\pgfqpoint{1.087cm}{1.3cm}}{\pgfqpoint{1.112cm}{1.274cm}}
\pgfpathcurveto{\pgfqpoint{1.138cm}{1.249cm}}{\pgfqpoint{1.172cm}{1.234cm}}{\pgfqpoint{1.209cm}{1.234cm}}
\pgfpathcurveto{\pgfqpoint{1.245cm}{1.234cm}}{\pgfqpoint{1.28cm}{1.249cm}}{\pgfqpoint{1.305cm}{1.274cm}}
\pgfpathcurveto{\pgfqpoint{1.331cm}{1.3cm}}{\pgfqpoint{1.345cm}{1.335cm}}{\pgfqpoint{1.345cm}{1.371cm}}
\pgfusepath{fill}
\begin{pgfscope}
\pgfsetdash{}{0cm}
\pgfsetlinewidth{0.818mm}
\pgfsetroundcap
\pgfsetmiterlimit{4.0}
\pgfpathmoveto{\pgfqpoint{0.682cm}{0.671cm}}
\pgfpathlineto{\pgfqpoint{0.682cm}{0.042cm}}
\pgfusepath{stroke}
\end{pgfscope}
\end{pgfscope}
\end{pgfscope}
\end{pgfscope}
\end{tikzpicture}}}(\phi+\psi)+3X(\phi+\psi)^2\\
&=3X\succ(X^{\!\resizebox{0.6em}{!}{
\begin{tikzpicture}
\pgfpathmoveto{\pgfqpoint{0cm}{-0.035cm}}
\pgfpathlineto{\pgfqpoint{1.376cm}{-0.035cm}}
\pgfpathlineto{\pgfqpoint{1.376cm}{1.552cm}}
\pgfpathlineto{\pgfqpoint{0cm}{1.552cm}}
\pgfpathclose
\pgfusepath{clip}
\begin{pgfscope}
\begin{pgfscope}
\pgfpathmoveto{\pgfqpoint{0cm}{-0.035cm}}
\pgfpathlineto{\pgfqpoint{1.376cm}{-0.035cm}}
\pgfpathlineto{\pgfqpoint{1.376cm}{1.552cm}}
\pgfpathlineto{\pgfqpoint{0cm}{1.552cm}}
\pgfpathclose
\pgfusepath{clip}
\begin{pgfscope}
\begin{pgfscope}
\pgfsetdash{}{0cm}
\pgfsetlinewidth{0.818mm}
\pgfsetroundcap
\pgfsetroundjoin
\pgfsetmiterlimit{7.0}
\definecolor{eps2pgf_color}{gray}{0}\pgfsetstrokecolor{eps2pgf_color}\pgfsetfillcolor{eps2pgf_color}
\pgfpathmoveto{\pgfqpoint{0.117cm}{1.421cm}}
\pgfpathlineto{\pgfqpoint{0.682cm}{0.671cm}}
\pgfpathlineto{\pgfqpoint{1.246cm}{1.421cm}}
\pgfusepath{stroke}
\end{pgfscope}
\definecolor{eps2pgf_color}{gray}{0}\pgfsetstrokecolor{eps2pgf_color}\pgfsetfillcolor{eps2pgf_color}
\pgfpathmoveto{\pgfqpoint{0.273cm}{1.395cm}}
\pgfpathcurveto{\pgfqpoint{0.273cm}{1.432cm}}{\pgfqpoint{0.259cm}{1.467cm}}{\pgfqpoint{0.233cm}{1.492cm}}
\pgfpathcurveto{\pgfqpoint{0.207cm}{1.518cm}}{\pgfqpoint{0.173cm}{1.532cm}}{\pgfqpoint{0.137cm}{1.532cm}}
\pgfpathcurveto{\pgfqpoint{0.1cm}{1.532cm}}{\pgfqpoint{0.066cm}{1.518cm}}{\pgfqpoint{0.04cm}{1.492cm}}
\pgfpathcurveto{\pgfqpoint{0.014cm}{1.467cm}}{\pgfqpoint{0cm}{1.432cm}}{\pgfqpoint{0cm}{1.395cm}}
\pgfpathcurveto{\pgfqpoint{0cm}{1.359cm}}{\pgfqpoint{0.014cm}{1.324cm}}{\pgfqpoint{0.04cm}{1.299cm}}
\pgfpathcurveto{\pgfqpoint{0.066cm}{1.273cm}}{\pgfqpoint{0.1cm}{1.258cm}}{\pgfqpoint{0.137cm}{1.258cm}}
\pgfpathcurveto{\pgfqpoint{0.173cm}{1.258cm}}{\pgfqpoint{0.207cm}{1.273cm}}{\pgfqpoint{0.233cm}{1.299cm}}
\pgfpathcurveto{\pgfqpoint{0.259cm}{1.324cm}}{\pgfqpoint{0.273cm}{1.359cm}}{\pgfqpoint{0.273cm}{1.395cm}}
\pgfusepath{fill}
\begin{pgfscope}
\pgfsetdash{}{0cm}
\pgfsetlinewidth{0.818mm}
\pgfsetmiterlimit{7.0}
\pgfpathmoveto{\pgfqpoint{0.682cm}{0.671cm}}
\pgfpathlineto{\pgfqpoint{0.679cm}{1.418cm}}
\pgfusepath{stroke}
\end{pgfscope}
\pgfpathmoveto{\pgfqpoint{0.815cm}{1.399cm}}
\pgfpathcurveto{\pgfqpoint{0.815cm}{1.435cm}}{\pgfqpoint{0.801cm}{1.47cm}}{\pgfqpoint{0.775cm}{1.496cm}}
\pgfpathcurveto{\pgfqpoint{0.75cm}{1.521cm}}{\pgfqpoint{0.715cm}{1.536cm}}{\pgfqpoint{0.679cm}{1.536cm}}
\pgfpathcurveto{\pgfqpoint{0.643cm}{1.536cm}}{\pgfqpoint{0.608cm}{1.521cm}}{\pgfqpoint{0.582cm}{1.496cm}}
\pgfpathcurveto{\pgfqpoint{0.557cm}{1.47cm}}{\pgfqpoint{0.542cm}{1.435cm}}{\pgfqpoint{0.542cm}{1.399cm}}
\pgfpathcurveto{\pgfqpoint{0.542cm}{1.363cm}}{\pgfqpoint{0.557cm}{1.328cm}}{\pgfqpoint{0.582cm}{1.302cm}}
\pgfpathcurveto{\pgfqpoint{0.608cm}{1.276cm}}{\pgfqpoint{0.643cm}{1.262cm}}{\pgfqpoint{0.679cm}{1.262cm}}
\pgfpathcurveto{\pgfqpoint{0.715cm}{1.262cm}}{\pgfqpoint{0.75cm}{1.276cm}}{\pgfqpoint{0.775cm}{1.302cm}}
\pgfpathcurveto{\pgfqpoint{0.801cm}{1.328cm}}{\pgfqpoint{0.815cm}{1.363cm}}{\pgfqpoint{0.815cm}{1.399cm}}
\pgfusepath{fill}
\pgfpathmoveto{\pgfqpoint{1.345cm}{1.371cm}}
\pgfpathcurveto{\pgfqpoint{1.345cm}{1.408cm}}{\pgfqpoint{1.331cm}{1.442cm}}{\pgfqpoint{1.305cm}{1.468cm}}
\pgfpathcurveto{\pgfqpoint{1.28cm}{1.494cm}}{\pgfqpoint{1.245cm}{1.508cm}}{\pgfqpoint{1.209cm}{1.508cm}}
\pgfpathcurveto{\pgfqpoint{1.172cm}{1.508cm}}{\pgfqpoint{1.138cm}{1.494cm}}{\pgfqpoint{1.112cm}{1.468cm}}
\pgfpathcurveto{\pgfqpoint{1.087cm}{1.442cm}}{\pgfqpoint{1.072cm}{1.408cm}}{\pgfqpoint{1.072cm}{1.371cm}}
\pgfpathcurveto{\pgfqpoint{1.072cm}{1.335cm}}{\pgfqpoint{1.087cm}{1.3cm}}{\pgfqpoint{1.112cm}{1.274cm}}
\pgfpathcurveto{\pgfqpoint{1.138cm}{1.249cm}}{\pgfqpoint{1.172cm}{1.234cm}}{\pgfqpoint{1.209cm}{1.234cm}}
\pgfpathcurveto{\pgfqpoint{1.245cm}{1.234cm}}{\pgfqpoint{1.28cm}{1.249cm}}{\pgfqpoint{1.305cm}{1.274cm}}
\pgfpathcurveto{\pgfqpoint{1.331cm}{1.3cm}}{\pgfqpoint{1.345cm}{1.335cm}}{\pgfqpoint{1.345cm}{1.371cm}}
\pgfusepath{fill}
\begin{pgfscope}
\pgfsetdash{}{0cm}
\pgfsetlinewidth{0.818mm}
\pgfsetroundcap
\pgfsetmiterlimit{4.0}
\pgfpathmoveto{\pgfqpoint{0.682cm}{0.671cm}}
\pgfpathlineto{\pgfqpoint{0.682cm}{0.042cm}}
\pgfusepath{stroke}
\end{pgfscope}
\end{pgfscope}
\end{pgfscope}
\end{pgfscope}
\end{tikzpicture}}})^2+3X\prec(X^{\!\resizebox{0.6em}{!}{
\begin{tikzpicture}
\pgfpathmoveto{\pgfqpoint{0cm}{-0.035cm}}
\pgfpathlineto{\pgfqpoint{1.376cm}{-0.035cm}}
\pgfpathlineto{\pgfqpoint{1.376cm}{1.552cm}}
\pgfpathlineto{\pgfqpoint{0cm}{1.552cm}}
\pgfpathclose
\pgfusepath{clip}
\begin{pgfscope}
\begin{pgfscope}
\pgfpathmoveto{\pgfqpoint{0cm}{-0.035cm}}
\pgfpathlineto{\pgfqpoint{1.376cm}{-0.035cm}}
\pgfpathlineto{\pgfqpoint{1.376cm}{1.552cm}}
\pgfpathlineto{\pgfqpoint{0cm}{1.552cm}}
\pgfpathclose
\pgfusepath{clip}
\begin{pgfscope}
\begin{pgfscope}
\pgfsetdash{}{0cm}
\pgfsetlinewidth{0.818mm}
\pgfsetroundcap
\pgfsetroundjoin
\pgfsetmiterlimit{7.0}
\definecolor{eps2pgf_color}{gray}{0}\pgfsetstrokecolor{eps2pgf_color}\pgfsetfillcolor{eps2pgf_color}
\pgfpathmoveto{\pgfqpoint{0.117cm}{1.421cm}}
\pgfpathlineto{\pgfqpoint{0.682cm}{0.671cm}}
\pgfpathlineto{\pgfqpoint{1.246cm}{1.421cm}}
\pgfusepath{stroke}
\end{pgfscope}
\definecolor{eps2pgf_color}{gray}{0}\pgfsetstrokecolor{eps2pgf_color}\pgfsetfillcolor{eps2pgf_color}
\pgfpathmoveto{\pgfqpoint{0.273cm}{1.395cm}}
\pgfpathcurveto{\pgfqpoint{0.273cm}{1.432cm}}{\pgfqpoint{0.259cm}{1.467cm}}{\pgfqpoint{0.233cm}{1.492cm}}
\pgfpathcurveto{\pgfqpoint{0.207cm}{1.518cm}}{\pgfqpoint{0.173cm}{1.532cm}}{\pgfqpoint{0.137cm}{1.532cm}}
\pgfpathcurveto{\pgfqpoint{0.1cm}{1.532cm}}{\pgfqpoint{0.066cm}{1.518cm}}{\pgfqpoint{0.04cm}{1.492cm}}
\pgfpathcurveto{\pgfqpoint{0.014cm}{1.467cm}}{\pgfqpoint{0cm}{1.432cm}}{\pgfqpoint{0cm}{1.395cm}}
\pgfpathcurveto{\pgfqpoint{0cm}{1.359cm}}{\pgfqpoint{0.014cm}{1.324cm}}{\pgfqpoint{0.04cm}{1.299cm}}
\pgfpathcurveto{\pgfqpoint{0.066cm}{1.273cm}}{\pgfqpoint{0.1cm}{1.258cm}}{\pgfqpoint{0.137cm}{1.258cm}}
\pgfpathcurveto{\pgfqpoint{0.173cm}{1.258cm}}{\pgfqpoint{0.207cm}{1.273cm}}{\pgfqpoint{0.233cm}{1.299cm}}
\pgfpathcurveto{\pgfqpoint{0.259cm}{1.324cm}}{\pgfqpoint{0.273cm}{1.359cm}}{\pgfqpoint{0.273cm}{1.395cm}}
\pgfusepath{fill}
\begin{pgfscope}
\pgfsetdash{}{0cm}
\pgfsetlinewidth{0.818mm}
\pgfsetmiterlimit{7.0}
\pgfpathmoveto{\pgfqpoint{0.682cm}{0.671cm}}
\pgfpathlineto{\pgfqpoint{0.679cm}{1.418cm}}
\pgfusepath{stroke}
\end{pgfscope}
\pgfpathmoveto{\pgfqpoint{0.815cm}{1.399cm}}
\pgfpathcurveto{\pgfqpoint{0.815cm}{1.435cm}}{\pgfqpoint{0.801cm}{1.47cm}}{\pgfqpoint{0.775cm}{1.496cm}}
\pgfpathcurveto{\pgfqpoint{0.75cm}{1.521cm}}{\pgfqpoint{0.715cm}{1.536cm}}{\pgfqpoint{0.679cm}{1.536cm}}
\pgfpathcurveto{\pgfqpoint{0.643cm}{1.536cm}}{\pgfqpoint{0.608cm}{1.521cm}}{\pgfqpoint{0.582cm}{1.496cm}}
\pgfpathcurveto{\pgfqpoint{0.557cm}{1.47cm}}{\pgfqpoint{0.542cm}{1.435cm}}{\pgfqpoint{0.542cm}{1.399cm}}
\pgfpathcurveto{\pgfqpoint{0.542cm}{1.363cm}}{\pgfqpoint{0.557cm}{1.328cm}}{\pgfqpoint{0.582cm}{1.302cm}}
\pgfpathcurveto{\pgfqpoint{0.608cm}{1.276cm}}{\pgfqpoint{0.643cm}{1.262cm}}{\pgfqpoint{0.679cm}{1.262cm}}
\pgfpathcurveto{\pgfqpoint{0.715cm}{1.262cm}}{\pgfqpoint{0.75cm}{1.276cm}}{\pgfqpoint{0.775cm}{1.302cm}}
\pgfpathcurveto{\pgfqpoint{0.801cm}{1.328cm}}{\pgfqpoint{0.815cm}{1.363cm}}{\pgfqpoint{0.815cm}{1.399cm}}
\pgfusepath{fill}
\pgfpathmoveto{\pgfqpoint{1.345cm}{1.371cm}}
\pgfpathcurveto{\pgfqpoint{1.345cm}{1.408cm}}{\pgfqpoint{1.331cm}{1.442cm}}{\pgfqpoint{1.305cm}{1.468cm}}
\pgfpathcurveto{\pgfqpoint{1.28cm}{1.494cm}}{\pgfqpoint{1.245cm}{1.508cm}}{\pgfqpoint{1.209cm}{1.508cm}}
\pgfpathcurveto{\pgfqpoint{1.172cm}{1.508cm}}{\pgfqpoint{1.138cm}{1.494cm}}{\pgfqpoint{1.112cm}{1.468cm}}
\pgfpathcurveto{\pgfqpoint{1.087cm}{1.442cm}}{\pgfqpoint{1.072cm}{1.408cm}}{\pgfqpoint{1.072cm}{1.371cm}}
\pgfpathcurveto{\pgfqpoint{1.072cm}{1.335cm}}{\pgfqpoint{1.087cm}{1.3cm}}{\pgfqpoint{1.112cm}{1.274cm}}
\pgfpathcurveto{\pgfqpoint{1.138cm}{1.249cm}}{\pgfqpoint{1.172cm}{1.234cm}}{\pgfqpoint{1.209cm}{1.234cm}}
\pgfpathcurveto{\pgfqpoint{1.245cm}{1.234cm}}{\pgfqpoint{1.28cm}{1.249cm}}{\pgfqpoint{1.305cm}{1.274cm}}
\pgfpathcurveto{\pgfqpoint{1.331cm}{1.3cm}}{\pgfqpoint{1.345cm}{1.335cm}}{\pgfqpoint{1.345cm}{1.371cm}}
\pgfusepath{fill}
\begin{pgfscope}
\pgfsetdash{}{0cm}
\pgfsetlinewidth{0.818mm}
\pgfsetroundcap
\pgfsetmiterlimit{4.0}
\pgfpathmoveto{\pgfqpoint{0.682cm}{0.671cm}}
\pgfpathlineto{\pgfqpoint{0.682cm}{0.042cm}}
\pgfusepath{stroke}
\end{pgfscope}
\end{pgfscope}
\end{pgfscope}
\end{pgfscope}
\end{tikzpicture}}})^2+3X\circ(X^{\!\resizebox{0.6em}{!}{
\begin{tikzpicture}
\pgfpathmoveto{\pgfqpoint{0cm}{-0.035cm}}
\pgfpathlineto{\pgfqpoint{1.376cm}{-0.035cm}}
\pgfpathlineto{\pgfqpoint{1.376cm}{1.552cm}}
\pgfpathlineto{\pgfqpoint{0cm}{1.552cm}}
\pgfpathclose
\pgfusepath{clip}
\begin{pgfscope}
\begin{pgfscope}
\pgfpathmoveto{\pgfqpoint{0cm}{-0.035cm}}
\pgfpathlineto{\pgfqpoint{1.376cm}{-0.035cm}}
\pgfpathlineto{\pgfqpoint{1.376cm}{1.552cm}}
\pgfpathlineto{\pgfqpoint{0cm}{1.552cm}}
\pgfpathclose
\pgfusepath{clip}
\begin{pgfscope}
\begin{pgfscope}
\pgfsetdash{}{0cm}
\pgfsetlinewidth{0.818mm}
\pgfsetroundcap
\pgfsetroundjoin
\pgfsetmiterlimit{7.0}
\definecolor{eps2pgf_color}{gray}{0}\pgfsetstrokecolor{eps2pgf_color}\pgfsetfillcolor{eps2pgf_color}
\pgfpathmoveto{\pgfqpoint{0.117cm}{1.421cm}}
\pgfpathlineto{\pgfqpoint{0.682cm}{0.671cm}}
\pgfpathlineto{\pgfqpoint{1.246cm}{1.421cm}}
\pgfusepath{stroke}
\end{pgfscope}
\definecolor{eps2pgf_color}{gray}{0}\pgfsetstrokecolor{eps2pgf_color}\pgfsetfillcolor{eps2pgf_color}
\pgfpathmoveto{\pgfqpoint{0.273cm}{1.395cm}}
\pgfpathcurveto{\pgfqpoint{0.273cm}{1.432cm}}{\pgfqpoint{0.259cm}{1.467cm}}{\pgfqpoint{0.233cm}{1.492cm}}
\pgfpathcurveto{\pgfqpoint{0.207cm}{1.518cm}}{\pgfqpoint{0.173cm}{1.532cm}}{\pgfqpoint{0.137cm}{1.532cm}}
\pgfpathcurveto{\pgfqpoint{0.1cm}{1.532cm}}{\pgfqpoint{0.066cm}{1.518cm}}{\pgfqpoint{0.04cm}{1.492cm}}
\pgfpathcurveto{\pgfqpoint{0.014cm}{1.467cm}}{\pgfqpoint{0cm}{1.432cm}}{\pgfqpoint{0cm}{1.395cm}}
\pgfpathcurveto{\pgfqpoint{0cm}{1.359cm}}{\pgfqpoint{0.014cm}{1.324cm}}{\pgfqpoint{0.04cm}{1.299cm}}
\pgfpathcurveto{\pgfqpoint{0.066cm}{1.273cm}}{\pgfqpoint{0.1cm}{1.258cm}}{\pgfqpoint{0.137cm}{1.258cm}}
\pgfpathcurveto{\pgfqpoint{0.173cm}{1.258cm}}{\pgfqpoint{0.207cm}{1.273cm}}{\pgfqpoint{0.233cm}{1.299cm}}
\pgfpathcurveto{\pgfqpoint{0.259cm}{1.324cm}}{\pgfqpoint{0.273cm}{1.359cm}}{\pgfqpoint{0.273cm}{1.395cm}}
\pgfusepath{fill}
\begin{pgfscope}
\pgfsetdash{}{0cm}
\pgfsetlinewidth{0.818mm}
\pgfsetmiterlimit{7.0}
\pgfpathmoveto{\pgfqpoint{0.682cm}{0.671cm}}
\pgfpathlineto{\pgfqpoint{0.679cm}{1.418cm}}
\pgfusepath{stroke}
\end{pgfscope}
\pgfpathmoveto{\pgfqpoint{0.815cm}{1.399cm}}
\pgfpathcurveto{\pgfqpoint{0.815cm}{1.435cm}}{\pgfqpoint{0.801cm}{1.47cm}}{\pgfqpoint{0.775cm}{1.496cm}}
\pgfpathcurveto{\pgfqpoint{0.75cm}{1.521cm}}{\pgfqpoint{0.715cm}{1.536cm}}{\pgfqpoint{0.679cm}{1.536cm}}
\pgfpathcurveto{\pgfqpoint{0.643cm}{1.536cm}}{\pgfqpoint{0.608cm}{1.521cm}}{\pgfqpoint{0.582cm}{1.496cm}}
\pgfpathcurveto{\pgfqpoint{0.557cm}{1.47cm}}{\pgfqpoint{0.542cm}{1.435cm}}{\pgfqpoint{0.542cm}{1.399cm}}
\pgfpathcurveto{\pgfqpoint{0.542cm}{1.363cm}}{\pgfqpoint{0.557cm}{1.328cm}}{\pgfqpoint{0.582cm}{1.302cm}}
\pgfpathcurveto{\pgfqpoint{0.608cm}{1.276cm}}{\pgfqpoint{0.643cm}{1.262cm}}{\pgfqpoint{0.679cm}{1.262cm}}
\pgfpathcurveto{\pgfqpoint{0.715cm}{1.262cm}}{\pgfqpoint{0.75cm}{1.276cm}}{\pgfqpoint{0.775cm}{1.302cm}}
\pgfpathcurveto{\pgfqpoint{0.801cm}{1.328cm}}{\pgfqpoint{0.815cm}{1.363cm}}{\pgfqpoint{0.815cm}{1.399cm}}
\pgfusepath{fill}
\pgfpathmoveto{\pgfqpoint{1.345cm}{1.371cm}}
\pgfpathcurveto{\pgfqpoint{1.345cm}{1.408cm}}{\pgfqpoint{1.331cm}{1.442cm}}{\pgfqpoint{1.305cm}{1.468cm}}
\pgfpathcurveto{\pgfqpoint{1.28cm}{1.494cm}}{\pgfqpoint{1.245cm}{1.508cm}}{\pgfqpoint{1.209cm}{1.508cm}}
\pgfpathcurveto{\pgfqpoint{1.172cm}{1.508cm}}{\pgfqpoint{1.138cm}{1.494cm}}{\pgfqpoint{1.112cm}{1.468cm}}
\pgfpathcurveto{\pgfqpoint{1.087cm}{1.442cm}}{\pgfqpoint{1.072cm}{1.408cm}}{\pgfqpoint{1.072cm}{1.371cm}}
\pgfpathcurveto{\pgfqpoint{1.072cm}{1.335cm}}{\pgfqpoint{1.087cm}{1.3cm}}{\pgfqpoint{1.112cm}{1.274cm}}
\pgfpathcurveto{\pgfqpoint{1.138cm}{1.249cm}}{\pgfqpoint{1.172cm}{1.234cm}}{\pgfqpoint{1.209cm}{1.234cm}}
\pgfpathcurveto{\pgfqpoint{1.245cm}{1.234cm}}{\pgfqpoint{1.28cm}{1.249cm}}{\pgfqpoint{1.305cm}{1.274cm}}
\pgfpathcurveto{\pgfqpoint{1.331cm}{1.3cm}}{\pgfqpoint{1.345cm}{1.335cm}}{\pgfqpoint{1.345cm}{1.371cm}}
\pgfusepath{fill}
\begin{pgfscope}
\pgfsetdash{}{0cm}
\pgfsetlinewidth{0.818mm}
\pgfsetroundcap
\pgfsetmiterlimit{4.0}
\pgfpathmoveto{\pgfqpoint{0.682cm}{0.671cm}}
\pgfpathlineto{\pgfqpoint{0.682cm}{0.042cm}}
\pgfusepath{stroke}
\end{pgfscope}
\end{pgfscope}
\end{pgfscope}
\end{pgfscope}
\end{tikzpicture}}})^2\\
&\quad-6X\succ(X^{\!\resizebox{0.6em}{!}{
\begin{tikzpicture}
\pgfpathmoveto{\pgfqpoint{0cm}{-0.035cm}}
\pgfpathlineto{\pgfqpoint{1.376cm}{-0.035cm}}
\pgfpathlineto{\pgfqpoint{1.376cm}{1.552cm}}
\pgfpathlineto{\pgfqpoint{0cm}{1.552cm}}
\pgfpathclose
\pgfusepath{clip}
\begin{pgfscope}
\begin{pgfscope}
\pgfpathmoveto{\pgfqpoint{0cm}{-0.035cm}}
\pgfpathlineto{\pgfqpoint{1.376cm}{-0.035cm}}
\pgfpathlineto{\pgfqpoint{1.376cm}{1.552cm}}
\pgfpathlineto{\pgfqpoint{0cm}{1.552cm}}
\pgfpathclose
\pgfusepath{clip}
\begin{pgfscope}
\begin{pgfscope}
\pgfsetdash{}{0cm}
\pgfsetlinewidth{0.818mm}
\pgfsetroundcap
\pgfsetroundjoin
\pgfsetmiterlimit{7.0}
\definecolor{eps2pgf_color}{gray}{0}\pgfsetstrokecolor{eps2pgf_color}\pgfsetfillcolor{eps2pgf_color}
\pgfpathmoveto{\pgfqpoint{0.117cm}{1.421cm}}
\pgfpathlineto{\pgfqpoint{0.682cm}{0.671cm}}
\pgfpathlineto{\pgfqpoint{1.246cm}{1.421cm}}
\pgfusepath{stroke}
\end{pgfscope}
\definecolor{eps2pgf_color}{gray}{0}\pgfsetstrokecolor{eps2pgf_color}\pgfsetfillcolor{eps2pgf_color}
\pgfpathmoveto{\pgfqpoint{0.273cm}{1.395cm}}
\pgfpathcurveto{\pgfqpoint{0.273cm}{1.432cm}}{\pgfqpoint{0.259cm}{1.467cm}}{\pgfqpoint{0.233cm}{1.492cm}}
\pgfpathcurveto{\pgfqpoint{0.207cm}{1.518cm}}{\pgfqpoint{0.173cm}{1.532cm}}{\pgfqpoint{0.137cm}{1.532cm}}
\pgfpathcurveto{\pgfqpoint{0.1cm}{1.532cm}}{\pgfqpoint{0.066cm}{1.518cm}}{\pgfqpoint{0.04cm}{1.492cm}}
\pgfpathcurveto{\pgfqpoint{0.014cm}{1.467cm}}{\pgfqpoint{0cm}{1.432cm}}{\pgfqpoint{0cm}{1.395cm}}
\pgfpathcurveto{\pgfqpoint{0cm}{1.359cm}}{\pgfqpoint{0.014cm}{1.324cm}}{\pgfqpoint{0.04cm}{1.299cm}}
\pgfpathcurveto{\pgfqpoint{0.066cm}{1.273cm}}{\pgfqpoint{0.1cm}{1.258cm}}{\pgfqpoint{0.137cm}{1.258cm}}
\pgfpathcurveto{\pgfqpoint{0.173cm}{1.258cm}}{\pgfqpoint{0.207cm}{1.273cm}}{\pgfqpoint{0.233cm}{1.299cm}}
\pgfpathcurveto{\pgfqpoint{0.259cm}{1.324cm}}{\pgfqpoint{0.273cm}{1.359cm}}{\pgfqpoint{0.273cm}{1.395cm}}
\pgfusepath{fill}
\begin{pgfscope}
\pgfsetdash{}{0cm}
\pgfsetlinewidth{0.818mm}
\pgfsetmiterlimit{7.0}
\pgfpathmoveto{\pgfqpoint{0.682cm}{0.671cm}}
\pgfpathlineto{\pgfqpoint{0.679cm}{1.418cm}}
\pgfusepath{stroke}
\end{pgfscope}
\pgfpathmoveto{\pgfqpoint{0.815cm}{1.399cm}}
\pgfpathcurveto{\pgfqpoint{0.815cm}{1.435cm}}{\pgfqpoint{0.801cm}{1.47cm}}{\pgfqpoint{0.775cm}{1.496cm}}
\pgfpathcurveto{\pgfqpoint{0.75cm}{1.521cm}}{\pgfqpoint{0.715cm}{1.536cm}}{\pgfqpoint{0.679cm}{1.536cm}}
\pgfpathcurveto{\pgfqpoint{0.643cm}{1.536cm}}{\pgfqpoint{0.608cm}{1.521cm}}{\pgfqpoint{0.582cm}{1.496cm}}
\pgfpathcurveto{\pgfqpoint{0.557cm}{1.47cm}}{\pgfqpoint{0.542cm}{1.435cm}}{\pgfqpoint{0.542cm}{1.399cm}}
\pgfpathcurveto{\pgfqpoint{0.542cm}{1.363cm}}{\pgfqpoint{0.557cm}{1.328cm}}{\pgfqpoint{0.582cm}{1.302cm}}
\pgfpathcurveto{\pgfqpoint{0.608cm}{1.276cm}}{\pgfqpoint{0.643cm}{1.262cm}}{\pgfqpoint{0.679cm}{1.262cm}}
\pgfpathcurveto{\pgfqpoint{0.715cm}{1.262cm}}{\pgfqpoint{0.75cm}{1.276cm}}{\pgfqpoint{0.775cm}{1.302cm}}
\pgfpathcurveto{\pgfqpoint{0.801cm}{1.328cm}}{\pgfqpoint{0.815cm}{1.363cm}}{\pgfqpoint{0.815cm}{1.399cm}}
\pgfusepath{fill}
\pgfpathmoveto{\pgfqpoint{1.345cm}{1.371cm}}
\pgfpathcurveto{\pgfqpoint{1.345cm}{1.408cm}}{\pgfqpoint{1.331cm}{1.442cm}}{\pgfqpoint{1.305cm}{1.468cm}}
\pgfpathcurveto{\pgfqpoint{1.28cm}{1.494cm}}{\pgfqpoint{1.245cm}{1.508cm}}{\pgfqpoint{1.209cm}{1.508cm}}
\pgfpathcurveto{\pgfqpoint{1.172cm}{1.508cm}}{\pgfqpoint{1.138cm}{1.494cm}}{\pgfqpoint{1.112cm}{1.468cm}}
\pgfpathcurveto{\pgfqpoint{1.087cm}{1.442cm}}{\pgfqpoint{1.072cm}{1.408cm}}{\pgfqpoint{1.072cm}{1.371cm}}
\pgfpathcurveto{\pgfqpoint{1.072cm}{1.335cm}}{\pgfqpoint{1.087cm}{1.3cm}}{\pgfqpoint{1.112cm}{1.274cm}}
\pgfpathcurveto{\pgfqpoint{1.138cm}{1.249cm}}{\pgfqpoint{1.172cm}{1.234cm}}{\pgfqpoint{1.209cm}{1.234cm}}
\pgfpathcurveto{\pgfqpoint{1.245cm}{1.234cm}}{\pgfqpoint{1.28cm}{1.249cm}}{\pgfqpoint{1.305cm}{1.274cm}}
\pgfpathcurveto{\pgfqpoint{1.331cm}{1.3cm}}{\pgfqpoint{1.345cm}{1.335cm}}{\pgfqpoint{1.345cm}{1.371cm}}
\pgfusepath{fill}
\begin{pgfscope}
\pgfsetdash{}{0cm}
\pgfsetlinewidth{0.818mm}
\pgfsetroundcap
\pgfsetmiterlimit{4.0}
\pgfpathmoveto{\pgfqpoint{0.682cm}{0.671cm}}
\pgfpathlineto{\pgfqpoint{0.682cm}{0.042cm}}
\pgfusepath{stroke}
\end{pgfscope}
\end{pgfscope}
\end{pgfscope}
\end{pgfscope}
\end{tikzpicture}}}(\phi+\psi))-6X\prec(X^{\!\resizebox{0.6em}{!}{
\begin{tikzpicture}
\pgfpathmoveto{\pgfqpoint{0cm}{-0.035cm}}
\pgfpathlineto{\pgfqpoint{1.376cm}{-0.035cm}}
\pgfpathlineto{\pgfqpoint{1.376cm}{1.552cm}}
\pgfpathlineto{\pgfqpoint{0cm}{1.552cm}}
\pgfpathclose
\pgfusepath{clip}
\begin{pgfscope}
\begin{pgfscope}
\pgfpathmoveto{\pgfqpoint{0cm}{-0.035cm}}
\pgfpathlineto{\pgfqpoint{1.376cm}{-0.035cm}}
\pgfpathlineto{\pgfqpoint{1.376cm}{1.552cm}}
\pgfpathlineto{\pgfqpoint{0cm}{1.552cm}}
\pgfpathclose
\pgfusepath{clip}
\begin{pgfscope}
\begin{pgfscope}
\pgfsetdash{}{0cm}
\pgfsetlinewidth{0.818mm}
\pgfsetroundcap
\pgfsetroundjoin
\pgfsetmiterlimit{7.0}
\definecolor{eps2pgf_color}{gray}{0}\pgfsetstrokecolor{eps2pgf_color}\pgfsetfillcolor{eps2pgf_color}
\pgfpathmoveto{\pgfqpoint{0.117cm}{1.421cm}}
\pgfpathlineto{\pgfqpoint{0.682cm}{0.671cm}}
\pgfpathlineto{\pgfqpoint{1.246cm}{1.421cm}}
\pgfusepath{stroke}
\end{pgfscope}
\definecolor{eps2pgf_color}{gray}{0}\pgfsetstrokecolor{eps2pgf_color}\pgfsetfillcolor{eps2pgf_color}
\pgfpathmoveto{\pgfqpoint{0.273cm}{1.395cm}}
\pgfpathcurveto{\pgfqpoint{0.273cm}{1.432cm}}{\pgfqpoint{0.259cm}{1.467cm}}{\pgfqpoint{0.233cm}{1.492cm}}
\pgfpathcurveto{\pgfqpoint{0.207cm}{1.518cm}}{\pgfqpoint{0.173cm}{1.532cm}}{\pgfqpoint{0.137cm}{1.532cm}}
\pgfpathcurveto{\pgfqpoint{0.1cm}{1.532cm}}{\pgfqpoint{0.066cm}{1.518cm}}{\pgfqpoint{0.04cm}{1.492cm}}
\pgfpathcurveto{\pgfqpoint{0.014cm}{1.467cm}}{\pgfqpoint{0cm}{1.432cm}}{\pgfqpoint{0cm}{1.395cm}}
\pgfpathcurveto{\pgfqpoint{0cm}{1.359cm}}{\pgfqpoint{0.014cm}{1.324cm}}{\pgfqpoint{0.04cm}{1.299cm}}
\pgfpathcurveto{\pgfqpoint{0.066cm}{1.273cm}}{\pgfqpoint{0.1cm}{1.258cm}}{\pgfqpoint{0.137cm}{1.258cm}}
\pgfpathcurveto{\pgfqpoint{0.173cm}{1.258cm}}{\pgfqpoint{0.207cm}{1.273cm}}{\pgfqpoint{0.233cm}{1.299cm}}
\pgfpathcurveto{\pgfqpoint{0.259cm}{1.324cm}}{\pgfqpoint{0.273cm}{1.359cm}}{\pgfqpoint{0.273cm}{1.395cm}}
\pgfusepath{fill}
\begin{pgfscope}
\pgfsetdash{}{0cm}
\pgfsetlinewidth{0.818mm}
\pgfsetmiterlimit{7.0}
\pgfpathmoveto{\pgfqpoint{0.682cm}{0.671cm}}
\pgfpathlineto{\pgfqpoint{0.679cm}{1.418cm}}
\pgfusepath{stroke}
\end{pgfscope}
\pgfpathmoveto{\pgfqpoint{0.815cm}{1.399cm}}
\pgfpathcurveto{\pgfqpoint{0.815cm}{1.435cm}}{\pgfqpoint{0.801cm}{1.47cm}}{\pgfqpoint{0.775cm}{1.496cm}}
\pgfpathcurveto{\pgfqpoint{0.75cm}{1.521cm}}{\pgfqpoint{0.715cm}{1.536cm}}{\pgfqpoint{0.679cm}{1.536cm}}
\pgfpathcurveto{\pgfqpoint{0.643cm}{1.536cm}}{\pgfqpoint{0.608cm}{1.521cm}}{\pgfqpoint{0.582cm}{1.496cm}}
\pgfpathcurveto{\pgfqpoint{0.557cm}{1.47cm}}{\pgfqpoint{0.542cm}{1.435cm}}{\pgfqpoint{0.542cm}{1.399cm}}
\pgfpathcurveto{\pgfqpoint{0.542cm}{1.363cm}}{\pgfqpoint{0.557cm}{1.328cm}}{\pgfqpoint{0.582cm}{1.302cm}}
\pgfpathcurveto{\pgfqpoint{0.608cm}{1.276cm}}{\pgfqpoint{0.643cm}{1.262cm}}{\pgfqpoint{0.679cm}{1.262cm}}
\pgfpathcurveto{\pgfqpoint{0.715cm}{1.262cm}}{\pgfqpoint{0.75cm}{1.276cm}}{\pgfqpoint{0.775cm}{1.302cm}}
\pgfpathcurveto{\pgfqpoint{0.801cm}{1.328cm}}{\pgfqpoint{0.815cm}{1.363cm}}{\pgfqpoint{0.815cm}{1.399cm}}
\pgfusepath{fill}
\pgfpathmoveto{\pgfqpoint{1.345cm}{1.371cm}}
\pgfpathcurveto{\pgfqpoint{1.345cm}{1.408cm}}{\pgfqpoint{1.331cm}{1.442cm}}{\pgfqpoint{1.305cm}{1.468cm}}
\pgfpathcurveto{\pgfqpoint{1.28cm}{1.494cm}}{\pgfqpoint{1.245cm}{1.508cm}}{\pgfqpoint{1.209cm}{1.508cm}}
\pgfpathcurveto{\pgfqpoint{1.172cm}{1.508cm}}{\pgfqpoint{1.138cm}{1.494cm}}{\pgfqpoint{1.112cm}{1.468cm}}
\pgfpathcurveto{\pgfqpoint{1.087cm}{1.442cm}}{\pgfqpoint{1.072cm}{1.408cm}}{\pgfqpoint{1.072cm}{1.371cm}}
\pgfpathcurveto{\pgfqpoint{1.072cm}{1.335cm}}{\pgfqpoint{1.087cm}{1.3cm}}{\pgfqpoint{1.112cm}{1.274cm}}
\pgfpathcurveto{\pgfqpoint{1.138cm}{1.249cm}}{\pgfqpoint{1.172cm}{1.234cm}}{\pgfqpoint{1.209cm}{1.234cm}}
\pgfpathcurveto{\pgfqpoint{1.245cm}{1.234cm}}{\pgfqpoint{1.28cm}{1.249cm}}{\pgfqpoint{1.305cm}{1.274cm}}
\pgfpathcurveto{\pgfqpoint{1.331cm}{1.3cm}}{\pgfqpoint{1.345cm}{1.335cm}}{\pgfqpoint{1.345cm}{1.371cm}}
\pgfusepath{fill}
\begin{pgfscope}
\pgfsetdash{}{0cm}
\pgfsetlinewidth{0.818mm}
\pgfsetroundcap
\pgfsetmiterlimit{4.0}
\pgfpathmoveto{\pgfqpoint{0.682cm}{0.671cm}}
\pgfpathlineto{\pgfqpoint{0.682cm}{0.042cm}}
\pgfusepath{stroke}
\end{pgfscope}
\end{pgfscope}
\end{pgfscope}
\end{pgfscope}
\end{tikzpicture}}}(\phi+\psi))-6X\circ(X^{\!\resizebox{0.6em}{!}{
\begin{tikzpicture}
\pgfpathmoveto{\pgfqpoint{0cm}{-0.035cm}}
\pgfpathlineto{\pgfqpoint{1.376cm}{-0.035cm}}
\pgfpathlineto{\pgfqpoint{1.376cm}{1.552cm}}
\pgfpathlineto{\pgfqpoint{0cm}{1.552cm}}
\pgfpathclose
\pgfusepath{clip}
\begin{pgfscope}
\begin{pgfscope}
\pgfpathmoveto{\pgfqpoint{0cm}{-0.035cm}}
\pgfpathlineto{\pgfqpoint{1.376cm}{-0.035cm}}
\pgfpathlineto{\pgfqpoint{1.376cm}{1.552cm}}
\pgfpathlineto{\pgfqpoint{0cm}{1.552cm}}
\pgfpathclose
\pgfusepath{clip}
\begin{pgfscope}
\begin{pgfscope}
\pgfsetdash{}{0cm}
\pgfsetlinewidth{0.818mm}
\pgfsetroundcap
\pgfsetroundjoin
\pgfsetmiterlimit{7.0}
\definecolor{eps2pgf_color}{gray}{0}\pgfsetstrokecolor{eps2pgf_color}\pgfsetfillcolor{eps2pgf_color}
\pgfpathmoveto{\pgfqpoint{0.117cm}{1.421cm}}
\pgfpathlineto{\pgfqpoint{0.682cm}{0.671cm}}
\pgfpathlineto{\pgfqpoint{1.246cm}{1.421cm}}
\pgfusepath{stroke}
\end{pgfscope}
\definecolor{eps2pgf_color}{gray}{0}\pgfsetstrokecolor{eps2pgf_color}\pgfsetfillcolor{eps2pgf_color}
\pgfpathmoveto{\pgfqpoint{0.273cm}{1.395cm}}
\pgfpathcurveto{\pgfqpoint{0.273cm}{1.432cm}}{\pgfqpoint{0.259cm}{1.467cm}}{\pgfqpoint{0.233cm}{1.492cm}}
\pgfpathcurveto{\pgfqpoint{0.207cm}{1.518cm}}{\pgfqpoint{0.173cm}{1.532cm}}{\pgfqpoint{0.137cm}{1.532cm}}
\pgfpathcurveto{\pgfqpoint{0.1cm}{1.532cm}}{\pgfqpoint{0.066cm}{1.518cm}}{\pgfqpoint{0.04cm}{1.492cm}}
\pgfpathcurveto{\pgfqpoint{0.014cm}{1.467cm}}{\pgfqpoint{0cm}{1.432cm}}{\pgfqpoint{0cm}{1.395cm}}
\pgfpathcurveto{\pgfqpoint{0cm}{1.359cm}}{\pgfqpoint{0.014cm}{1.324cm}}{\pgfqpoint{0.04cm}{1.299cm}}
\pgfpathcurveto{\pgfqpoint{0.066cm}{1.273cm}}{\pgfqpoint{0.1cm}{1.258cm}}{\pgfqpoint{0.137cm}{1.258cm}}
\pgfpathcurveto{\pgfqpoint{0.173cm}{1.258cm}}{\pgfqpoint{0.207cm}{1.273cm}}{\pgfqpoint{0.233cm}{1.299cm}}
\pgfpathcurveto{\pgfqpoint{0.259cm}{1.324cm}}{\pgfqpoint{0.273cm}{1.359cm}}{\pgfqpoint{0.273cm}{1.395cm}}
\pgfusepath{fill}
\begin{pgfscope}
\pgfsetdash{}{0cm}
\pgfsetlinewidth{0.818mm}
\pgfsetmiterlimit{7.0}
\pgfpathmoveto{\pgfqpoint{0.682cm}{0.671cm}}
\pgfpathlineto{\pgfqpoint{0.679cm}{1.418cm}}
\pgfusepath{stroke}
\end{pgfscope}
\pgfpathmoveto{\pgfqpoint{0.815cm}{1.399cm}}
\pgfpathcurveto{\pgfqpoint{0.815cm}{1.435cm}}{\pgfqpoint{0.801cm}{1.47cm}}{\pgfqpoint{0.775cm}{1.496cm}}
\pgfpathcurveto{\pgfqpoint{0.75cm}{1.521cm}}{\pgfqpoint{0.715cm}{1.536cm}}{\pgfqpoint{0.679cm}{1.536cm}}
\pgfpathcurveto{\pgfqpoint{0.643cm}{1.536cm}}{\pgfqpoint{0.608cm}{1.521cm}}{\pgfqpoint{0.582cm}{1.496cm}}
\pgfpathcurveto{\pgfqpoint{0.557cm}{1.47cm}}{\pgfqpoint{0.542cm}{1.435cm}}{\pgfqpoint{0.542cm}{1.399cm}}
\pgfpathcurveto{\pgfqpoint{0.542cm}{1.363cm}}{\pgfqpoint{0.557cm}{1.328cm}}{\pgfqpoint{0.582cm}{1.302cm}}
\pgfpathcurveto{\pgfqpoint{0.608cm}{1.276cm}}{\pgfqpoint{0.643cm}{1.262cm}}{\pgfqpoint{0.679cm}{1.262cm}}
\pgfpathcurveto{\pgfqpoint{0.715cm}{1.262cm}}{\pgfqpoint{0.75cm}{1.276cm}}{\pgfqpoint{0.775cm}{1.302cm}}
\pgfpathcurveto{\pgfqpoint{0.801cm}{1.328cm}}{\pgfqpoint{0.815cm}{1.363cm}}{\pgfqpoint{0.815cm}{1.399cm}}
\pgfusepath{fill}
\pgfpathmoveto{\pgfqpoint{1.345cm}{1.371cm}}
\pgfpathcurveto{\pgfqpoint{1.345cm}{1.408cm}}{\pgfqpoint{1.331cm}{1.442cm}}{\pgfqpoint{1.305cm}{1.468cm}}
\pgfpathcurveto{\pgfqpoint{1.28cm}{1.494cm}}{\pgfqpoint{1.245cm}{1.508cm}}{\pgfqpoint{1.209cm}{1.508cm}}
\pgfpathcurveto{\pgfqpoint{1.172cm}{1.508cm}}{\pgfqpoint{1.138cm}{1.494cm}}{\pgfqpoint{1.112cm}{1.468cm}}
\pgfpathcurveto{\pgfqpoint{1.087cm}{1.442cm}}{\pgfqpoint{1.072cm}{1.408cm}}{\pgfqpoint{1.072cm}{1.371cm}}
\pgfpathcurveto{\pgfqpoint{1.072cm}{1.335cm}}{\pgfqpoint{1.087cm}{1.3cm}}{\pgfqpoint{1.112cm}{1.274cm}}
\pgfpathcurveto{\pgfqpoint{1.138cm}{1.249cm}}{\pgfqpoint{1.172cm}{1.234cm}}{\pgfqpoint{1.209cm}{1.234cm}}
\pgfpathcurveto{\pgfqpoint{1.245cm}{1.234cm}}{\pgfqpoint{1.28cm}{1.249cm}}{\pgfqpoint{1.305cm}{1.274cm}}
\pgfpathcurveto{\pgfqpoint{1.331cm}{1.3cm}}{\pgfqpoint{1.345cm}{1.335cm}}{\pgfqpoint{1.345cm}{1.371cm}}
\pgfusepath{fill}
\begin{pgfscope}
\pgfsetdash{}{0cm}
\pgfsetlinewidth{0.818mm}
\pgfsetroundcap
\pgfsetmiterlimit{4.0}
\pgfpathmoveto{\pgfqpoint{0.682cm}{0.671cm}}
\pgfpathlineto{\pgfqpoint{0.682cm}{0.042cm}}
\pgfusepath{stroke}
\end{pgfscope}
\end{pgfscope}
\end{pgfscope}
\end{pgfscope}
\end{tikzpicture}}}(\phi+\psi))\\
&\quad+3X(\phi+\psi)^2,
\end{align*}
where
$$
3X\circ(X^{\!\resizebox{0.6em}{!}{
\begin{tikzpicture}
\pgfpathmoveto{\pgfqpoint{0cm}{-0.035cm}}
\pgfpathlineto{\pgfqpoint{1.376cm}{-0.035cm}}
\pgfpathlineto{\pgfqpoint{1.376cm}{1.552cm}}
\pgfpathlineto{\pgfqpoint{0cm}{1.552cm}}
\pgfpathclose
\pgfusepath{clip}
\begin{pgfscope}
\begin{pgfscope}
\pgfpathmoveto{\pgfqpoint{0cm}{-0.035cm}}
\pgfpathlineto{\pgfqpoint{1.376cm}{-0.035cm}}
\pgfpathlineto{\pgfqpoint{1.376cm}{1.552cm}}
\pgfpathlineto{\pgfqpoint{0cm}{1.552cm}}
\pgfpathclose
\pgfusepath{clip}
\begin{pgfscope}
\begin{pgfscope}
\pgfsetdash{}{0cm}
\pgfsetlinewidth{0.818mm}
\pgfsetroundcap
\pgfsetroundjoin
\pgfsetmiterlimit{7.0}
\definecolor{eps2pgf_color}{gray}{0}\pgfsetstrokecolor{eps2pgf_color}\pgfsetfillcolor{eps2pgf_color}
\pgfpathmoveto{\pgfqpoint{0.117cm}{1.421cm}}
\pgfpathlineto{\pgfqpoint{0.682cm}{0.671cm}}
\pgfpathlineto{\pgfqpoint{1.246cm}{1.421cm}}
\pgfusepath{stroke}
\end{pgfscope}
\definecolor{eps2pgf_color}{gray}{0}\pgfsetstrokecolor{eps2pgf_color}\pgfsetfillcolor{eps2pgf_color}
\pgfpathmoveto{\pgfqpoint{0.273cm}{1.395cm}}
\pgfpathcurveto{\pgfqpoint{0.273cm}{1.432cm}}{\pgfqpoint{0.259cm}{1.467cm}}{\pgfqpoint{0.233cm}{1.492cm}}
\pgfpathcurveto{\pgfqpoint{0.207cm}{1.518cm}}{\pgfqpoint{0.173cm}{1.532cm}}{\pgfqpoint{0.137cm}{1.532cm}}
\pgfpathcurveto{\pgfqpoint{0.1cm}{1.532cm}}{\pgfqpoint{0.066cm}{1.518cm}}{\pgfqpoint{0.04cm}{1.492cm}}
\pgfpathcurveto{\pgfqpoint{0.014cm}{1.467cm}}{\pgfqpoint{0cm}{1.432cm}}{\pgfqpoint{0cm}{1.395cm}}
\pgfpathcurveto{\pgfqpoint{0cm}{1.359cm}}{\pgfqpoint{0.014cm}{1.324cm}}{\pgfqpoint{0.04cm}{1.299cm}}
\pgfpathcurveto{\pgfqpoint{0.066cm}{1.273cm}}{\pgfqpoint{0.1cm}{1.258cm}}{\pgfqpoint{0.137cm}{1.258cm}}
\pgfpathcurveto{\pgfqpoint{0.173cm}{1.258cm}}{\pgfqpoint{0.207cm}{1.273cm}}{\pgfqpoint{0.233cm}{1.299cm}}
\pgfpathcurveto{\pgfqpoint{0.259cm}{1.324cm}}{\pgfqpoint{0.273cm}{1.359cm}}{\pgfqpoint{0.273cm}{1.395cm}}
\pgfusepath{fill}
\begin{pgfscope}
\pgfsetdash{}{0cm}
\pgfsetlinewidth{0.818mm}
\pgfsetmiterlimit{7.0}
\pgfpathmoveto{\pgfqpoint{0.682cm}{0.671cm}}
\pgfpathlineto{\pgfqpoint{0.679cm}{1.418cm}}
\pgfusepath{stroke}
\end{pgfscope}
\pgfpathmoveto{\pgfqpoint{0.815cm}{1.399cm}}
\pgfpathcurveto{\pgfqpoint{0.815cm}{1.435cm}}{\pgfqpoint{0.801cm}{1.47cm}}{\pgfqpoint{0.775cm}{1.496cm}}
\pgfpathcurveto{\pgfqpoint{0.75cm}{1.521cm}}{\pgfqpoint{0.715cm}{1.536cm}}{\pgfqpoint{0.679cm}{1.536cm}}
\pgfpathcurveto{\pgfqpoint{0.643cm}{1.536cm}}{\pgfqpoint{0.608cm}{1.521cm}}{\pgfqpoint{0.582cm}{1.496cm}}
\pgfpathcurveto{\pgfqpoint{0.557cm}{1.47cm}}{\pgfqpoint{0.542cm}{1.435cm}}{\pgfqpoint{0.542cm}{1.399cm}}
\pgfpathcurveto{\pgfqpoint{0.542cm}{1.363cm}}{\pgfqpoint{0.557cm}{1.328cm}}{\pgfqpoint{0.582cm}{1.302cm}}
\pgfpathcurveto{\pgfqpoint{0.608cm}{1.276cm}}{\pgfqpoint{0.643cm}{1.262cm}}{\pgfqpoint{0.679cm}{1.262cm}}
\pgfpathcurveto{\pgfqpoint{0.715cm}{1.262cm}}{\pgfqpoint{0.75cm}{1.276cm}}{\pgfqpoint{0.775cm}{1.302cm}}
\pgfpathcurveto{\pgfqpoint{0.801cm}{1.328cm}}{\pgfqpoint{0.815cm}{1.363cm}}{\pgfqpoint{0.815cm}{1.399cm}}
\pgfusepath{fill}
\pgfpathmoveto{\pgfqpoint{1.345cm}{1.371cm}}
\pgfpathcurveto{\pgfqpoint{1.345cm}{1.408cm}}{\pgfqpoint{1.331cm}{1.442cm}}{\pgfqpoint{1.305cm}{1.468cm}}
\pgfpathcurveto{\pgfqpoint{1.28cm}{1.494cm}}{\pgfqpoint{1.245cm}{1.508cm}}{\pgfqpoint{1.209cm}{1.508cm}}
\pgfpathcurveto{\pgfqpoint{1.172cm}{1.508cm}}{\pgfqpoint{1.138cm}{1.494cm}}{\pgfqpoint{1.112cm}{1.468cm}}
\pgfpathcurveto{\pgfqpoint{1.087cm}{1.442cm}}{\pgfqpoint{1.072cm}{1.408cm}}{\pgfqpoint{1.072cm}{1.371cm}}
\pgfpathcurveto{\pgfqpoint{1.072cm}{1.335cm}}{\pgfqpoint{1.087cm}{1.3cm}}{\pgfqpoint{1.112cm}{1.274cm}}
\pgfpathcurveto{\pgfqpoint{1.138cm}{1.249cm}}{\pgfqpoint{1.172cm}{1.234cm}}{\pgfqpoint{1.209cm}{1.234cm}}
\pgfpathcurveto{\pgfqpoint{1.245cm}{1.234cm}}{\pgfqpoint{1.28cm}{1.249cm}}{\pgfqpoint{1.305cm}{1.274cm}}
\pgfpathcurveto{\pgfqpoint{1.331cm}{1.3cm}}{\pgfqpoint{1.345cm}{1.335cm}}{\pgfqpoint{1.345cm}{1.371cm}}
\pgfusepath{fill}
\begin{pgfscope}
\pgfsetdash{}{0cm}
\pgfsetlinewidth{0.818mm}
\pgfsetroundcap
\pgfsetmiterlimit{4.0}
\pgfpathmoveto{\pgfqpoint{0.682cm}{0.671cm}}
\pgfpathlineto{\pgfqpoint{0.682cm}{0.042cm}}
\pgfusepath{stroke}
\end{pgfscope}
\end{pgfscope}
\end{pgfscope}
\end{pgfscope}
\end{tikzpicture}}})^2=6X\circ(X^{\!\resizebox{0.6em}{!}{
\begin{tikzpicture}
\pgfpathmoveto{\pgfqpoint{0cm}{-0.035cm}}
\pgfpathlineto{\pgfqpoint{1.376cm}{-0.035cm}}
\pgfpathlineto{\pgfqpoint{1.376cm}{1.552cm}}
\pgfpathlineto{\pgfqpoint{0cm}{1.552cm}}
\pgfpathclose
\pgfusepath{clip}
\begin{pgfscope}
\begin{pgfscope}
\pgfpathmoveto{\pgfqpoint{0cm}{-0.035cm}}
\pgfpathlineto{\pgfqpoint{1.376cm}{-0.035cm}}
\pgfpathlineto{\pgfqpoint{1.376cm}{1.552cm}}
\pgfpathlineto{\pgfqpoint{0cm}{1.552cm}}
\pgfpathclose
\pgfusepath{clip}
\begin{pgfscope}
\begin{pgfscope}
\pgfsetdash{}{0cm}
\pgfsetlinewidth{0.818mm}
\pgfsetroundcap
\pgfsetroundjoin
\pgfsetmiterlimit{7.0}
\definecolor{eps2pgf_color}{gray}{0}\pgfsetstrokecolor{eps2pgf_color}\pgfsetfillcolor{eps2pgf_color}
\pgfpathmoveto{\pgfqpoint{0.117cm}{1.421cm}}
\pgfpathlineto{\pgfqpoint{0.682cm}{0.671cm}}
\pgfpathlineto{\pgfqpoint{1.246cm}{1.421cm}}
\pgfusepath{stroke}
\end{pgfscope}
\definecolor{eps2pgf_color}{gray}{0}\pgfsetstrokecolor{eps2pgf_color}\pgfsetfillcolor{eps2pgf_color}
\pgfpathmoveto{\pgfqpoint{0.273cm}{1.395cm}}
\pgfpathcurveto{\pgfqpoint{0.273cm}{1.432cm}}{\pgfqpoint{0.259cm}{1.467cm}}{\pgfqpoint{0.233cm}{1.492cm}}
\pgfpathcurveto{\pgfqpoint{0.207cm}{1.518cm}}{\pgfqpoint{0.173cm}{1.532cm}}{\pgfqpoint{0.137cm}{1.532cm}}
\pgfpathcurveto{\pgfqpoint{0.1cm}{1.532cm}}{\pgfqpoint{0.066cm}{1.518cm}}{\pgfqpoint{0.04cm}{1.492cm}}
\pgfpathcurveto{\pgfqpoint{0.014cm}{1.467cm}}{\pgfqpoint{0cm}{1.432cm}}{\pgfqpoint{0cm}{1.395cm}}
\pgfpathcurveto{\pgfqpoint{0cm}{1.359cm}}{\pgfqpoint{0.014cm}{1.324cm}}{\pgfqpoint{0.04cm}{1.299cm}}
\pgfpathcurveto{\pgfqpoint{0.066cm}{1.273cm}}{\pgfqpoint{0.1cm}{1.258cm}}{\pgfqpoint{0.137cm}{1.258cm}}
\pgfpathcurveto{\pgfqpoint{0.173cm}{1.258cm}}{\pgfqpoint{0.207cm}{1.273cm}}{\pgfqpoint{0.233cm}{1.299cm}}
\pgfpathcurveto{\pgfqpoint{0.259cm}{1.324cm}}{\pgfqpoint{0.273cm}{1.359cm}}{\pgfqpoint{0.273cm}{1.395cm}}
\pgfusepath{fill}
\begin{pgfscope}
\pgfsetdash{}{0cm}
\pgfsetlinewidth{0.818mm}
\pgfsetmiterlimit{7.0}
\pgfpathmoveto{\pgfqpoint{0.682cm}{0.671cm}}
\pgfpathlineto{\pgfqpoint{0.679cm}{1.418cm}}
\pgfusepath{stroke}
\end{pgfscope}
\pgfpathmoveto{\pgfqpoint{0.815cm}{1.399cm}}
\pgfpathcurveto{\pgfqpoint{0.815cm}{1.435cm}}{\pgfqpoint{0.801cm}{1.47cm}}{\pgfqpoint{0.775cm}{1.496cm}}
\pgfpathcurveto{\pgfqpoint{0.75cm}{1.521cm}}{\pgfqpoint{0.715cm}{1.536cm}}{\pgfqpoint{0.679cm}{1.536cm}}
\pgfpathcurveto{\pgfqpoint{0.643cm}{1.536cm}}{\pgfqpoint{0.608cm}{1.521cm}}{\pgfqpoint{0.582cm}{1.496cm}}
\pgfpathcurveto{\pgfqpoint{0.557cm}{1.47cm}}{\pgfqpoint{0.542cm}{1.435cm}}{\pgfqpoint{0.542cm}{1.399cm}}
\pgfpathcurveto{\pgfqpoint{0.542cm}{1.363cm}}{\pgfqpoint{0.557cm}{1.328cm}}{\pgfqpoint{0.582cm}{1.302cm}}
\pgfpathcurveto{\pgfqpoint{0.608cm}{1.276cm}}{\pgfqpoint{0.643cm}{1.262cm}}{\pgfqpoint{0.679cm}{1.262cm}}
\pgfpathcurveto{\pgfqpoint{0.715cm}{1.262cm}}{\pgfqpoint{0.75cm}{1.276cm}}{\pgfqpoint{0.775cm}{1.302cm}}
\pgfpathcurveto{\pgfqpoint{0.801cm}{1.328cm}}{\pgfqpoint{0.815cm}{1.363cm}}{\pgfqpoint{0.815cm}{1.399cm}}
\pgfusepath{fill}
\pgfpathmoveto{\pgfqpoint{1.345cm}{1.371cm}}
\pgfpathcurveto{\pgfqpoint{1.345cm}{1.408cm}}{\pgfqpoint{1.331cm}{1.442cm}}{\pgfqpoint{1.305cm}{1.468cm}}
\pgfpathcurveto{\pgfqpoint{1.28cm}{1.494cm}}{\pgfqpoint{1.245cm}{1.508cm}}{\pgfqpoint{1.209cm}{1.508cm}}
\pgfpathcurveto{\pgfqpoint{1.172cm}{1.508cm}}{\pgfqpoint{1.138cm}{1.494cm}}{\pgfqpoint{1.112cm}{1.468cm}}
\pgfpathcurveto{\pgfqpoint{1.087cm}{1.442cm}}{\pgfqpoint{1.072cm}{1.408cm}}{\pgfqpoint{1.072cm}{1.371cm}}
\pgfpathcurveto{\pgfqpoint{1.072cm}{1.335cm}}{\pgfqpoint{1.087cm}{1.3cm}}{\pgfqpoint{1.112cm}{1.274cm}}
\pgfpathcurveto{\pgfqpoint{1.138cm}{1.249cm}}{\pgfqpoint{1.172cm}{1.234cm}}{\pgfqpoint{1.209cm}{1.234cm}}
\pgfpathcurveto{\pgfqpoint{1.245cm}{1.234cm}}{\pgfqpoint{1.28cm}{1.249cm}}{\pgfqpoint{1.305cm}{1.274cm}}
\pgfpathcurveto{\pgfqpoint{1.331cm}{1.3cm}}{\pgfqpoint{1.345cm}{1.335cm}}{\pgfqpoint{1.345cm}{1.371cm}}
\pgfusepath{fill}
\begin{pgfscope}
\pgfsetdash{}{0cm}
\pgfsetlinewidth{0.818mm}
\pgfsetroundcap
\pgfsetmiterlimit{4.0}
\pgfpathmoveto{\pgfqpoint{0.682cm}{0.671cm}}
\pgfpathlineto{\pgfqpoint{0.682cm}{0.042cm}}
\pgfusepath{stroke}
\end{pgfscope}
\end{pgfscope}
\end{pgfscope}
\end{pgfscope}
\end{tikzpicture}}}\prec X^{\!\resizebox{0.6em}{!}{
\begin{tikzpicture}
\pgfpathmoveto{\pgfqpoint{0cm}{-0.035cm}}
\pgfpathlineto{\pgfqpoint{1.376cm}{-0.035cm}}
\pgfpathlineto{\pgfqpoint{1.376cm}{1.552cm}}
\pgfpathlineto{\pgfqpoint{0cm}{1.552cm}}
\pgfpathclose
\pgfusepath{clip}
\begin{pgfscope}
\begin{pgfscope}
\pgfpathmoveto{\pgfqpoint{0cm}{-0.035cm}}
\pgfpathlineto{\pgfqpoint{1.376cm}{-0.035cm}}
\pgfpathlineto{\pgfqpoint{1.376cm}{1.552cm}}
\pgfpathlineto{\pgfqpoint{0cm}{1.552cm}}
\pgfpathclose
\pgfusepath{clip}
\begin{pgfscope}
\begin{pgfscope}
\pgfsetdash{}{0cm}
\pgfsetlinewidth{0.818mm}
\pgfsetroundcap
\pgfsetroundjoin
\pgfsetmiterlimit{7.0}
\definecolor{eps2pgf_color}{gray}{0}\pgfsetstrokecolor{eps2pgf_color}\pgfsetfillcolor{eps2pgf_color}
\pgfpathmoveto{\pgfqpoint{0.117cm}{1.421cm}}
\pgfpathlineto{\pgfqpoint{0.682cm}{0.671cm}}
\pgfpathlineto{\pgfqpoint{1.246cm}{1.421cm}}
\pgfusepath{stroke}
\end{pgfscope}
\definecolor{eps2pgf_color}{gray}{0}\pgfsetstrokecolor{eps2pgf_color}\pgfsetfillcolor{eps2pgf_color}
\pgfpathmoveto{\pgfqpoint{0.273cm}{1.395cm}}
\pgfpathcurveto{\pgfqpoint{0.273cm}{1.432cm}}{\pgfqpoint{0.259cm}{1.467cm}}{\pgfqpoint{0.233cm}{1.492cm}}
\pgfpathcurveto{\pgfqpoint{0.207cm}{1.518cm}}{\pgfqpoint{0.173cm}{1.532cm}}{\pgfqpoint{0.137cm}{1.532cm}}
\pgfpathcurveto{\pgfqpoint{0.1cm}{1.532cm}}{\pgfqpoint{0.066cm}{1.518cm}}{\pgfqpoint{0.04cm}{1.492cm}}
\pgfpathcurveto{\pgfqpoint{0.014cm}{1.467cm}}{\pgfqpoint{0cm}{1.432cm}}{\pgfqpoint{0cm}{1.395cm}}
\pgfpathcurveto{\pgfqpoint{0cm}{1.359cm}}{\pgfqpoint{0.014cm}{1.324cm}}{\pgfqpoint{0.04cm}{1.299cm}}
\pgfpathcurveto{\pgfqpoint{0.066cm}{1.273cm}}{\pgfqpoint{0.1cm}{1.258cm}}{\pgfqpoint{0.137cm}{1.258cm}}
\pgfpathcurveto{\pgfqpoint{0.173cm}{1.258cm}}{\pgfqpoint{0.207cm}{1.273cm}}{\pgfqpoint{0.233cm}{1.299cm}}
\pgfpathcurveto{\pgfqpoint{0.259cm}{1.324cm}}{\pgfqpoint{0.273cm}{1.359cm}}{\pgfqpoint{0.273cm}{1.395cm}}
\pgfusepath{fill}
\begin{pgfscope}
\pgfsetdash{}{0cm}
\pgfsetlinewidth{0.818mm}
\pgfsetmiterlimit{7.0}
\pgfpathmoveto{\pgfqpoint{0.682cm}{0.671cm}}
\pgfpathlineto{\pgfqpoint{0.679cm}{1.418cm}}
\pgfusepath{stroke}
\end{pgfscope}
\pgfpathmoveto{\pgfqpoint{0.815cm}{1.399cm}}
\pgfpathcurveto{\pgfqpoint{0.815cm}{1.435cm}}{\pgfqpoint{0.801cm}{1.47cm}}{\pgfqpoint{0.775cm}{1.496cm}}
\pgfpathcurveto{\pgfqpoint{0.75cm}{1.521cm}}{\pgfqpoint{0.715cm}{1.536cm}}{\pgfqpoint{0.679cm}{1.536cm}}
\pgfpathcurveto{\pgfqpoint{0.643cm}{1.536cm}}{\pgfqpoint{0.608cm}{1.521cm}}{\pgfqpoint{0.582cm}{1.496cm}}
\pgfpathcurveto{\pgfqpoint{0.557cm}{1.47cm}}{\pgfqpoint{0.542cm}{1.435cm}}{\pgfqpoint{0.542cm}{1.399cm}}
\pgfpathcurveto{\pgfqpoint{0.542cm}{1.363cm}}{\pgfqpoint{0.557cm}{1.328cm}}{\pgfqpoint{0.582cm}{1.302cm}}
\pgfpathcurveto{\pgfqpoint{0.608cm}{1.276cm}}{\pgfqpoint{0.643cm}{1.262cm}}{\pgfqpoint{0.679cm}{1.262cm}}
\pgfpathcurveto{\pgfqpoint{0.715cm}{1.262cm}}{\pgfqpoint{0.75cm}{1.276cm}}{\pgfqpoint{0.775cm}{1.302cm}}
\pgfpathcurveto{\pgfqpoint{0.801cm}{1.328cm}}{\pgfqpoint{0.815cm}{1.363cm}}{\pgfqpoint{0.815cm}{1.399cm}}
\pgfusepath{fill}
\pgfpathmoveto{\pgfqpoint{1.345cm}{1.371cm}}
\pgfpathcurveto{\pgfqpoint{1.345cm}{1.408cm}}{\pgfqpoint{1.331cm}{1.442cm}}{\pgfqpoint{1.305cm}{1.468cm}}
\pgfpathcurveto{\pgfqpoint{1.28cm}{1.494cm}}{\pgfqpoint{1.245cm}{1.508cm}}{\pgfqpoint{1.209cm}{1.508cm}}
\pgfpathcurveto{\pgfqpoint{1.172cm}{1.508cm}}{\pgfqpoint{1.138cm}{1.494cm}}{\pgfqpoint{1.112cm}{1.468cm}}
\pgfpathcurveto{\pgfqpoint{1.087cm}{1.442cm}}{\pgfqpoint{1.072cm}{1.408cm}}{\pgfqpoint{1.072cm}{1.371cm}}
\pgfpathcurveto{\pgfqpoint{1.072cm}{1.335cm}}{\pgfqpoint{1.087cm}{1.3cm}}{\pgfqpoint{1.112cm}{1.274cm}}
\pgfpathcurveto{\pgfqpoint{1.138cm}{1.249cm}}{\pgfqpoint{1.172cm}{1.234cm}}{\pgfqpoint{1.209cm}{1.234cm}}
\pgfpathcurveto{\pgfqpoint{1.245cm}{1.234cm}}{\pgfqpoint{1.28cm}{1.249cm}}{\pgfqpoint{1.305cm}{1.274cm}}
\pgfpathcurveto{\pgfqpoint{1.331cm}{1.3cm}}{\pgfqpoint{1.345cm}{1.335cm}}{\pgfqpoint{1.345cm}{1.371cm}}
\pgfusepath{fill}
\begin{pgfscope}
\pgfsetdash{}{0cm}
\pgfsetlinewidth{0.818mm}
\pgfsetroundcap
\pgfsetmiterlimit{4.0}
\pgfpathmoveto{\pgfqpoint{0.682cm}{0.671cm}}
\pgfpathlineto{\pgfqpoint{0.682cm}{0.042cm}}
\pgfusepath{stroke}
\end{pgfscope}
\end{pgfscope}
\end{pgfscope}
\end{pgfscope}
\end{tikzpicture}}})+3X\circ (X^{\!\resizebox{0.6em}{!}{
\begin{tikzpicture}
\pgfpathmoveto{\pgfqpoint{0cm}{-0.035cm}}
\pgfpathlineto{\pgfqpoint{1.376cm}{-0.035cm}}
\pgfpathlineto{\pgfqpoint{1.376cm}{1.552cm}}
\pgfpathlineto{\pgfqpoint{0cm}{1.552cm}}
\pgfpathclose
\pgfusepath{clip}
\begin{pgfscope}
\begin{pgfscope}
\pgfpathmoveto{\pgfqpoint{0cm}{-0.035cm}}
\pgfpathlineto{\pgfqpoint{1.376cm}{-0.035cm}}
\pgfpathlineto{\pgfqpoint{1.376cm}{1.552cm}}
\pgfpathlineto{\pgfqpoint{0cm}{1.552cm}}
\pgfpathclose
\pgfusepath{clip}
\begin{pgfscope}
\begin{pgfscope}
\pgfsetdash{}{0cm}
\pgfsetlinewidth{0.818mm}
\pgfsetroundcap
\pgfsetroundjoin
\pgfsetmiterlimit{7.0}
\definecolor{eps2pgf_color}{gray}{0}\pgfsetstrokecolor{eps2pgf_color}\pgfsetfillcolor{eps2pgf_color}
\pgfpathmoveto{\pgfqpoint{0.117cm}{1.421cm}}
\pgfpathlineto{\pgfqpoint{0.682cm}{0.671cm}}
\pgfpathlineto{\pgfqpoint{1.246cm}{1.421cm}}
\pgfusepath{stroke}
\end{pgfscope}
\definecolor{eps2pgf_color}{gray}{0}\pgfsetstrokecolor{eps2pgf_color}\pgfsetfillcolor{eps2pgf_color}
\pgfpathmoveto{\pgfqpoint{0.273cm}{1.395cm}}
\pgfpathcurveto{\pgfqpoint{0.273cm}{1.432cm}}{\pgfqpoint{0.259cm}{1.467cm}}{\pgfqpoint{0.233cm}{1.492cm}}
\pgfpathcurveto{\pgfqpoint{0.207cm}{1.518cm}}{\pgfqpoint{0.173cm}{1.532cm}}{\pgfqpoint{0.137cm}{1.532cm}}
\pgfpathcurveto{\pgfqpoint{0.1cm}{1.532cm}}{\pgfqpoint{0.066cm}{1.518cm}}{\pgfqpoint{0.04cm}{1.492cm}}
\pgfpathcurveto{\pgfqpoint{0.014cm}{1.467cm}}{\pgfqpoint{0cm}{1.432cm}}{\pgfqpoint{0cm}{1.395cm}}
\pgfpathcurveto{\pgfqpoint{0cm}{1.359cm}}{\pgfqpoint{0.014cm}{1.324cm}}{\pgfqpoint{0.04cm}{1.299cm}}
\pgfpathcurveto{\pgfqpoint{0.066cm}{1.273cm}}{\pgfqpoint{0.1cm}{1.258cm}}{\pgfqpoint{0.137cm}{1.258cm}}
\pgfpathcurveto{\pgfqpoint{0.173cm}{1.258cm}}{\pgfqpoint{0.207cm}{1.273cm}}{\pgfqpoint{0.233cm}{1.299cm}}
\pgfpathcurveto{\pgfqpoint{0.259cm}{1.324cm}}{\pgfqpoint{0.273cm}{1.359cm}}{\pgfqpoint{0.273cm}{1.395cm}}
\pgfusepath{fill}
\begin{pgfscope}
\pgfsetdash{}{0cm}
\pgfsetlinewidth{0.818mm}
\pgfsetmiterlimit{7.0}
\pgfpathmoveto{\pgfqpoint{0.682cm}{0.671cm}}
\pgfpathlineto{\pgfqpoint{0.679cm}{1.418cm}}
\pgfusepath{stroke}
\end{pgfscope}
\pgfpathmoveto{\pgfqpoint{0.815cm}{1.399cm}}
\pgfpathcurveto{\pgfqpoint{0.815cm}{1.435cm}}{\pgfqpoint{0.801cm}{1.47cm}}{\pgfqpoint{0.775cm}{1.496cm}}
\pgfpathcurveto{\pgfqpoint{0.75cm}{1.521cm}}{\pgfqpoint{0.715cm}{1.536cm}}{\pgfqpoint{0.679cm}{1.536cm}}
\pgfpathcurveto{\pgfqpoint{0.643cm}{1.536cm}}{\pgfqpoint{0.608cm}{1.521cm}}{\pgfqpoint{0.582cm}{1.496cm}}
\pgfpathcurveto{\pgfqpoint{0.557cm}{1.47cm}}{\pgfqpoint{0.542cm}{1.435cm}}{\pgfqpoint{0.542cm}{1.399cm}}
\pgfpathcurveto{\pgfqpoint{0.542cm}{1.363cm}}{\pgfqpoint{0.557cm}{1.328cm}}{\pgfqpoint{0.582cm}{1.302cm}}
\pgfpathcurveto{\pgfqpoint{0.608cm}{1.276cm}}{\pgfqpoint{0.643cm}{1.262cm}}{\pgfqpoint{0.679cm}{1.262cm}}
\pgfpathcurveto{\pgfqpoint{0.715cm}{1.262cm}}{\pgfqpoint{0.75cm}{1.276cm}}{\pgfqpoint{0.775cm}{1.302cm}}
\pgfpathcurveto{\pgfqpoint{0.801cm}{1.328cm}}{\pgfqpoint{0.815cm}{1.363cm}}{\pgfqpoint{0.815cm}{1.399cm}}
\pgfusepath{fill}
\pgfpathmoveto{\pgfqpoint{1.345cm}{1.371cm}}
\pgfpathcurveto{\pgfqpoint{1.345cm}{1.408cm}}{\pgfqpoint{1.331cm}{1.442cm}}{\pgfqpoint{1.305cm}{1.468cm}}
\pgfpathcurveto{\pgfqpoint{1.28cm}{1.494cm}}{\pgfqpoint{1.245cm}{1.508cm}}{\pgfqpoint{1.209cm}{1.508cm}}
\pgfpathcurveto{\pgfqpoint{1.172cm}{1.508cm}}{\pgfqpoint{1.138cm}{1.494cm}}{\pgfqpoint{1.112cm}{1.468cm}}
\pgfpathcurveto{\pgfqpoint{1.087cm}{1.442cm}}{\pgfqpoint{1.072cm}{1.408cm}}{\pgfqpoint{1.072cm}{1.371cm}}
\pgfpathcurveto{\pgfqpoint{1.072cm}{1.335cm}}{\pgfqpoint{1.087cm}{1.3cm}}{\pgfqpoint{1.112cm}{1.274cm}}
\pgfpathcurveto{\pgfqpoint{1.138cm}{1.249cm}}{\pgfqpoint{1.172cm}{1.234cm}}{\pgfqpoint{1.209cm}{1.234cm}}
\pgfpathcurveto{\pgfqpoint{1.245cm}{1.234cm}}{\pgfqpoint{1.28cm}{1.249cm}}{\pgfqpoint{1.305cm}{1.274cm}}
\pgfpathcurveto{\pgfqpoint{1.331cm}{1.3cm}}{\pgfqpoint{1.345cm}{1.335cm}}{\pgfqpoint{1.345cm}{1.371cm}}
\pgfusepath{fill}
\begin{pgfscope}
\pgfsetdash{}{0cm}
\pgfsetlinewidth{0.818mm}
\pgfsetroundcap
\pgfsetmiterlimit{4.0}
\pgfpathmoveto{\pgfqpoint{0.682cm}{0.671cm}}
\pgfpathlineto{\pgfqpoint{0.682cm}{0.042cm}}
\pgfusepath{stroke}
\end{pgfscope}
\end{pgfscope}
\end{pgfscope}
\end{pgfscope}
\end{tikzpicture}}}\circ X^{\!\resizebox{0.6em}{!}{
\begin{tikzpicture}
\pgfpathmoveto{\pgfqpoint{0cm}{-0.035cm}}
\pgfpathlineto{\pgfqpoint{1.376cm}{-0.035cm}}
\pgfpathlineto{\pgfqpoint{1.376cm}{1.552cm}}
\pgfpathlineto{\pgfqpoint{0cm}{1.552cm}}
\pgfpathclose
\pgfusepath{clip}
\begin{pgfscope}
\begin{pgfscope}
\pgfpathmoveto{\pgfqpoint{0cm}{-0.035cm}}
\pgfpathlineto{\pgfqpoint{1.376cm}{-0.035cm}}
\pgfpathlineto{\pgfqpoint{1.376cm}{1.552cm}}
\pgfpathlineto{\pgfqpoint{0cm}{1.552cm}}
\pgfpathclose
\pgfusepath{clip}
\begin{pgfscope}
\begin{pgfscope}
\pgfsetdash{}{0cm}
\pgfsetlinewidth{0.818mm}
\pgfsetroundcap
\pgfsetroundjoin
\pgfsetmiterlimit{7.0}
\definecolor{eps2pgf_color}{gray}{0}\pgfsetstrokecolor{eps2pgf_color}\pgfsetfillcolor{eps2pgf_color}
\pgfpathmoveto{\pgfqpoint{0.117cm}{1.421cm}}
\pgfpathlineto{\pgfqpoint{0.682cm}{0.671cm}}
\pgfpathlineto{\pgfqpoint{1.246cm}{1.421cm}}
\pgfusepath{stroke}
\end{pgfscope}
\definecolor{eps2pgf_color}{gray}{0}\pgfsetstrokecolor{eps2pgf_color}\pgfsetfillcolor{eps2pgf_color}
\pgfpathmoveto{\pgfqpoint{0.273cm}{1.395cm}}
\pgfpathcurveto{\pgfqpoint{0.273cm}{1.432cm}}{\pgfqpoint{0.259cm}{1.467cm}}{\pgfqpoint{0.233cm}{1.492cm}}
\pgfpathcurveto{\pgfqpoint{0.207cm}{1.518cm}}{\pgfqpoint{0.173cm}{1.532cm}}{\pgfqpoint{0.137cm}{1.532cm}}
\pgfpathcurveto{\pgfqpoint{0.1cm}{1.532cm}}{\pgfqpoint{0.066cm}{1.518cm}}{\pgfqpoint{0.04cm}{1.492cm}}
\pgfpathcurveto{\pgfqpoint{0.014cm}{1.467cm}}{\pgfqpoint{0cm}{1.432cm}}{\pgfqpoint{0cm}{1.395cm}}
\pgfpathcurveto{\pgfqpoint{0cm}{1.359cm}}{\pgfqpoint{0.014cm}{1.324cm}}{\pgfqpoint{0.04cm}{1.299cm}}
\pgfpathcurveto{\pgfqpoint{0.066cm}{1.273cm}}{\pgfqpoint{0.1cm}{1.258cm}}{\pgfqpoint{0.137cm}{1.258cm}}
\pgfpathcurveto{\pgfqpoint{0.173cm}{1.258cm}}{\pgfqpoint{0.207cm}{1.273cm}}{\pgfqpoint{0.233cm}{1.299cm}}
\pgfpathcurveto{\pgfqpoint{0.259cm}{1.324cm}}{\pgfqpoint{0.273cm}{1.359cm}}{\pgfqpoint{0.273cm}{1.395cm}}
\pgfusepath{fill}
\begin{pgfscope}
\pgfsetdash{}{0cm}
\pgfsetlinewidth{0.818mm}
\pgfsetmiterlimit{7.0}
\pgfpathmoveto{\pgfqpoint{0.682cm}{0.671cm}}
\pgfpathlineto{\pgfqpoint{0.679cm}{1.418cm}}
\pgfusepath{stroke}
\end{pgfscope}
\pgfpathmoveto{\pgfqpoint{0.815cm}{1.399cm}}
\pgfpathcurveto{\pgfqpoint{0.815cm}{1.435cm}}{\pgfqpoint{0.801cm}{1.47cm}}{\pgfqpoint{0.775cm}{1.496cm}}
\pgfpathcurveto{\pgfqpoint{0.75cm}{1.521cm}}{\pgfqpoint{0.715cm}{1.536cm}}{\pgfqpoint{0.679cm}{1.536cm}}
\pgfpathcurveto{\pgfqpoint{0.643cm}{1.536cm}}{\pgfqpoint{0.608cm}{1.521cm}}{\pgfqpoint{0.582cm}{1.496cm}}
\pgfpathcurveto{\pgfqpoint{0.557cm}{1.47cm}}{\pgfqpoint{0.542cm}{1.435cm}}{\pgfqpoint{0.542cm}{1.399cm}}
\pgfpathcurveto{\pgfqpoint{0.542cm}{1.363cm}}{\pgfqpoint{0.557cm}{1.328cm}}{\pgfqpoint{0.582cm}{1.302cm}}
\pgfpathcurveto{\pgfqpoint{0.608cm}{1.276cm}}{\pgfqpoint{0.643cm}{1.262cm}}{\pgfqpoint{0.679cm}{1.262cm}}
\pgfpathcurveto{\pgfqpoint{0.715cm}{1.262cm}}{\pgfqpoint{0.75cm}{1.276cm}}{\pgfqpoint{0.775cm}{1.302cm}}
\pgfpathcurveto{\pgfqpoint{0.801cm}{1.328cm}}{\pgfqpoint{0.815cm}{1.363cm}}{\pgfqpoint{0.815cm}{1.399cm}}
\pgfusepath{fill}
\pgfpathmoveto{\pgfqpoint{1.345cm}{1.371cm}}
\pgfpathcurveto{\pgfqpoint{1.345cm}{1.408cm}}{\pgfqpoint{1.331cm}{1.442cm}}{\pgfqpoint{1.305cm}{1.468cm}}
\pgfpathcurveto{\pgfqpoint{1.28cm}{1.494cm}}{\pgfqpoint{1.245cm}{1.508cm}}{\pgfqpoint{1.209cm}{1.508cm}}
\pgfpathcurveto{\pgfqpoint{1.172cm}{1.508cm}}{\pgfqpoint{1.138cm}{1.494cm}}{\pgfqpoint{1.112cm}{1.468cm}}
\pgfpathcurveto{\pgfqpoint{1.087cm}{1.442cm}}{\pgfqpoint{1.072cm}{1.408cm}}{\pgfqpoint{1.072cm}{1.371cm}}
\pgfpathcurveto{\pgfqpoint{1.072cm}{1.335cm}}{\pgfqpoint{1.087cm}{1.3cm}}{\pgfqpoint{1.112cm}{1.274cm}}
\pgfpathcurveto{\pgfqpoint{1.138cm}{1.249cm}}{\pgfqpoint{1.172cm}{1.234cm}}{\pgfqpoint{1.209cm}{1.234cm}}
\pgfpathcurveto{\pgfqpoint{1.245cm}{1.234cm}}{\pgfqpoint{1.28cm}{1.249cm}}{\pgfqpoint{1.305cm}{1.274cm}}
\pgfpathcurveto{\pgfqpoint{1.331cm}{1.3cm}}{\pgfqpoint{1.345cm}{1.335cm}}{\pgfqpoint{1.345cm}{1.371cm}}
\pgfusepath{fill}
\begin{pgfscope}
\pgfsetdash{}{0cm}
\pgfsetlinewidth{0.818mm}
\pgfsetroundcap
\pgfsetmiterlimit{4.0}
\pgfpathmoveto{\pgfqpoint{0.682cm}{0.671cm}}
\pgfpathlineto{\pgfqpoint{0.682cm}{0.042cm}}
\pgfusepath{stroke}
\end{pgfscope}
\end{pgfscope}
\end{pgfscope}
\end{pgfscope}
\end{tikzpicture}}})=6X^{\!\resizebox{0.6em}{!}{
\begin{tikzpicture}
\pgfpathmoveto{\pgfqpoint{0cm}{-0.035cm}}
\pgfpathlineto{\pgfqpoint{1.376cm}{-0.035cm}}
\pgfpathlineto{\pgfqpoint{1.376cm}{1.552cm}}
\pgfpathlineto{\pgfqpoint{0cm}{1.552cm}}
\pgfpathclose
\pgfusepath{clip}
\begin{pgfscope}
\begin{pgfscope}
\pgfpathmoveto{\pgfqpoint{0cm}{-0.035cm}}
\pgfpathlineto{\pgfqpoint{1.376cm}{-0.035cm}}
\pgfpathlineto{\pgfqpoint{1.376cm}{1.552cm}}
\pgfpathlineto{\pgfqpoint{0cm}{1.552cm}}
\pgfpathclose
\pgfusepath{clip}
\begin{pgfscope}
\begin{pgfscope}
\pgfsetdash{}{0cm}
\pgfsetlinewidth{0.818mm}
\pgfsetroundcap
\pgfsetroundjoin
\pgfsetmiterlimit{7.0}
\definecolor{eps2pgf_color}{gray}{0}\pgfsetstrokecolor{eps2pgf_color}\pgfsetfillcolor{eps2pgf_color}
\pgfpathmoveto{\pgfqpoint{0.117cm}{1.421cm}}
\pgfpathlineto{\pgfqpoint{0.682cm}{0.671cm}}
\pgfpathlineto{\pgfqpoint{1.246cm}{1.421cm}}
\pgfusepath{stroke}
\end{pgfscope}
\definecolor{eps2pgf_color}{gray}{0}\pgfsetstrokecolor{eps2pgf_color}\pgfsetfillcolor{eps2pgf_color}
\pgfpathmoveto{\pgfqpoint{0.273cm}{1.395cm}}
\pgfpathcurveto{\pgfqpoint{0.273cm}{1.432cm}}{\pgfqpoint{0.259cm}{1.467cm}}{\pgfqpoint{0.233cm}{1.492cm}}
\pgfpathcurveto{\pgfqpoint{0.207cm}{1.518cm}}{\pgfqpoint{0.173cm}{1.532cm}}{\pgfqpoint{0.137cm}{1.532cm}}
\pgfpathcurveto{\pgfqpoint{0.1cm}{1.532cm}}{\pgfqpoint{0.066cm}{1.518cm}}{\pgfqpoint{0.04cm}{1.492cm}}
\pgfpathcurveto{\pgfqpoint{0.014cm}{1.467cm}}{\pgfqpoint{0cm}{1.432cm}}{\pgfqpoint{0cm}{1.395cm}}
\pgfpathcurveto{\pgfqpoint{0cm}{1.359cm}}{\pgfqpoint{0.014cm}{1.324cm}}{\pgfqpoint{0.04cm}{1.299cm}}
\pgfpathcurveto{\pgfqpoint{0.066cm}{1.273cm}}{\pgfqpoint{0.1cm}{1.258cm}}{\pgfqpoint{0.137cm}{1.258cm}}
\pgfpathcurveto{\pgfqpoint{0.173cm}{1.258cm}}{\pgfqpoint{0.207cm}{1.273cm}}{\pgfqpoint{0.233cm}{1.299cm}}
\pgfpathcurveto{\pgfqpoint{0.259cm}{1.324cm}}{\pgfqpoint{0.273cm}{1.359cm}}{\pgfqpoint{0.273cm}{1.395cm}}
\pgfusepath{fill}
\begin{pgfscope}
\pgfsetdash{}{0cm}
\pgfsetlinewidth{0.818mm}
\pgfsetmiterlimit{7.0}
\pgfpathmoveto{\pgfqpoint{0.682cm}{0.671cm}}
\pgfpathlineto{\pgfqpoint{0.679cm}{1.418cm}}
\pgfusepath{stroke}
\end{pgfscope}
\pgfpathmoveto{\pgfqpoint{0.815cm}{1.399cm}}
\pgfpathcurveto{\pgfqpoint{0.815cm}{1.435cm}}{\pgfqpoint{0.801cm}{1.47cm}}{\pgfqpoint{0.775cm}{1.496cm}}
\pgfpathcurveto{\pgfqpoint{0.75cm}{1.521cm}}{\pgfqpoint{0.715cm}{1.536cm}}{\pgfqpoint{0.679cm}{1.536cm}}
\pgfpathcurveto{\pgfqpoint{0.643cm}{1.536cm}}{\pgfqpoint{0.608cm}{1.521cm}}{\pgfqpoint{0.582cm}{1.496cm}}
\pgfpathcurveto{\pgfqpoint{0.557cm}{1.47cm}}{\pgfqpoint{0.542cm}{1.435cm}}{\pgfqpoint{0.542cm}{1.399cm}}
\pgfpathcurveto{\pgfqpoint{0.542cm}{1.363cm}}{\pgfqpoint{0.557cm}{1.328cm}}{\pgfqpoint{0.582cm}{1.302cm}}
\pgfpathcurveto{\pgfqpoint{0.608cm}{1.276cm}}{\pgfqpoint{0.643cm}{1.262cm}}{\pgfqpoint{0.679cm}{1.262cm}}
\pgfpathcurveto{\pgfqpoint{0.715cm}{1.262cm}}{\pgfqpoint{0.75cm}{1.276cm}}{\pgfqpoint{0.775cm}{1.302cm}}
\pgfpathcurveto{\pgfqpoint{0.801cm}{1.328cm}}{\pgfqpoint{0.815cm}{1.363cm}}{\pgfqpoint{0.815cm}{1.399cm}}
\pgfusepath{fill}
\pgfpathmoveto{\pgfqpoint{1.345cm}{1.371cm}}
\pgfpathcurveto{\pgfqpoint{1.345cm}{1.408cm}}{\pgfqpoint{1.331cm}{1.442cm}}{\pgfqpoint{1.305cm}{1.468cm}}
\pgfpathcurveto{\pgfqpoint{1.28cm}{1.494cm}}{\pgfqpoint{1.245cm}{1.508cm}}{\pgfqpoint{1.209cm}{1.508cm}}
\pgfpathcurveto{\pgfqpoint{1.172cm}{1.508cm}}{\pgfqpoint{1.138cm}{1.494cm}}{\pgfqpoint{1.112cm}{1.468cm}}
\pgfpathcurveto{\pgfqpoint{1.087cm}{1.442cm}}{\pgfqpoint{1.072cm}{1.408cm}}{\pgfqpoint{1.072cm}{1.371cm}}
\pgfpathcurveto{\pgfqpoint{1.072cm}{1.335cm}}{\pgfqpoint{1.087cm}{1.3cm}}{\pgfqpoint{1.112cm}{1.274cm}}
\pgfpathcurveto{\pgfqpoint{1.138cm}{1.249cm}}{\pgfqpoint{1.172cm}{1.234cm}}{\pgfqpoint{1.209cm}{1.234cm}}
\pgfpathcurveto{\pgfqpoint{1.245cm}{1.234cm}}{\pgfqpoint{1.28cm}{1.249cm}}{\pgfqpoint{1.305cm}{1.274cm}}
\pgfpathcurveto{\pgfqpoint{1.331cm}{1.3cm}}{\pgfqpoint{1.345cm}{1.335cm}}{\pgfqpoint{1.345cm}{1.371cm}}
\pgfusepath{fill}
\begin{pgfscope}
\pgfsetdash{}{0cm}
\pgfsetlinewidth{0.818mm}
\pgfsetroundcap
\pgfsetmiterlimit{4.0}
\pgfpathmoveto{\pgfqpoint{0.682cm}{0.671cm}}
\pgfpathlineto{\pgfqpoint{0.682cm}{0.042cm}}
\pgfusepath{stroke}
\end{pgfscope}
\end{pgfscope}
\end{pgfscope}
\end{pgfscope}
\end{tikzpicture}}}X^{\!\resizebox{!}{.8em}{
\begin{tikzpicture}
\pgfpathmoveto{\pgfqpoint{0cm}{-0.035cm}}
\pgfpathlineto{\pgfqpoint{1.976cm}{-0.035cm}}
\pgfpathlineto{\pgfqpoint{1.976cm}{1.94cm}}
\pgfpathlineto{\pgfqpoint{0cm}{1.94cm}}
\pgfpathclose
\pgfusepath{clip}
\begin{pgfscope}
\begin{pgfscope}
\pgfpathmoveto{\pgfqpoint{0cm}{-0.035cm}}
\pgfpathlineto{\pgfqpoint{1.976cm}{-0.035cm}}
\pgfpathlineto{\pgfqpoint{1.976cm}{1.94cm}}
\pgfpathlineto{\pgfqpoint{0cm}{1.94cm}}
\pgfpathclose
\pgfusepath{clip}
\begin{pgfscope}
\begin{pgfscope}
\pgfsetdash{}{0cm}
\pgfsetlinewidth{0.818mm}
\pgfsetroundcap
\pgfsetroundjoin
\pgfsetmiterlimit{7.0}
\definecolor{eps2pgf_color}{gray}{0}\pgfsetstrokecolor{eps2pgf_color}\pgfsetfillcolor{eps2pgf_color}
\pgfpathmoveto{\pgfqpoint{0.117cm}{1.815cm}}
\pgfpathlineto{\pgfqpoint{0.682cm}{1.065cm}}
\pgfpathlineto{\pgfqpoint{1.246cm}{1.815cm}}
\pgfusepath{stroke}
\end{pgfscope}
\definecolor{eps2pgf_color}{gray}{0}\pgfsetstrokecolor{eps2pgf_color}\pgfsetfillcolor{eps2pgf_color}
\pgfpathmoveto{\pgfqpoint{0.273cm}{1.789cm}}
\pgfpathcurveto{\pgfqpoint{0.273cm}{1.825cm}}{\pgfqpoint{0.259cm}{1.86cm}}{\pgfqpoint{0.233cm}{1.886cm}}
\pgfpathcurveto{\pgfqpoint{0.207cm}{1.912cm}}{\pgfqpoint{0.173cm}{1.926cm}}{\pgfqpoint{0.137cm}{1.926cm}}
\pgfpathcurveto{\pgfqpoint{0.1cm}{1.926cm}}{\pgfqpoint{0.066cm}{1.912cm}}{\pgfqpoint{0.04cm}{1.886cm}}
\pgfpathcurveto{\pgfqpoint{0.014cm}{1.86cm}}{\pgfqpoint{0cm}{1.825cm}}{\pgfqpoint{0cm}{1.789cm}}
\pgfpathcurveto{\pgfqpoint{0cm}{1.753cm}}{\pgfqpoint{0.014cm}{1.718cm}}{\pgfqpoint{0.04cm}{1.692cm}}
\pgfpathcurveto{\pgfqpoint{0.066cm}{1.667cm}}{\pgfqpoint{0.1cm}{1.652cm}}{\pgfqpoint{0.137cm}{1.652cm}}
\pgfpathcurveto{\pgfqpoint{0.173cm}{1.652cm}}{\pgfqpoint{0.207cm}{1.667cm}}{\pgfqpoint{0.233cm}{1.692cm}}
\pgfpathcurveto{\pgfqpoint{0.259cm}{1.718cm}}{\pgfqpoint{0.273cm}{1.753cm}}{\pgfqpoint{0.273cm}{1.789cm}}
\pgfusepath{fill}
\begin{pgfscope}
\pgfsetdash{}{0cm}
\pgfsetlinewidth{0.818mm}
\pgfsetmiterlimit{7.0}
\pgfpathmoveto{\pgfqpoint{0.682cm}{1.065cm}}
\pgfpathlineto{\pgfqpoint{0.679cm}{1.812cm}}
\pgfusepath{stroke}
\end{pgfscope}
\pgfpathmoveto{\pgfqpoint{0.815cm}{1.793cm}}
\pgfpathcurveto{\pgfqpoint{0.815cm}{1.829cm}}{\pgfqpoint{0.801cm}{1.864cm}}{\pgfqpoint{0.775cm}{1.89cm}}
\pgfpathcurveto{\pgfqpoint{0.75cm}{1.915cm}}{\pgfqpoint{0.715cm}{1.93cm}}{\pgfqpoint{0.679cm}{1.93cm}}
\pgfpathcurveto{\pgfqpoint{0.643cm}{1.93cm}}{\pgfqpoint{0.608cm}{1.915cm}}{\pgfqpoint{0.582cm}{1.89cm}}
\pgfpathcurveto{\pgfqpoint{0.557cm}{1.864cm}}{\pgfqpoint{0.542cm}{1.829cm}}{\pgfqpoint{0.542cm}{1.793cm}}
\pgfpathcurveto{\pgfqpoint{0.542cm}{1.756cm}}{\pgfqpoint{0.557cm}{1.722cm}}{\pgfqpoint{0.582cm}{1.696cm}}
\pgfpathcurveto{\pgfqpoint{0.608cm}{1.67cm}}{\pgfqpoint{0.643cm}{1.656cm}}{\pgfqpoint{0.679cm}{1.656cm}}
\pgfpathcurveto{\pgfqpoint{0.715cm}{1.656cm}}{\pgfqpoint{0.75cm}{1.67cm}}{\pgfqpoint{0.775cm}{1.696cm}}
\pgfpathcurveto{\pgfqpoint{0.801cm}{1.722cm}}{\pgfqpoint{0.815cm}{1.756cm}}{\pgfqpoint{0.815cm}{1.793cm}}
\pgfusepath{fill}
\pgfpathmoveto{\pgfqpoint{1.345cm}{1.765cm}}
\pgfpathcurveto{\pgfqpoint{1.345cm}{1.801cm}}{\pgfqpoint{1.331cm}{1.836cm}}{\pgfqpoint{1.305cm}{1.862cm}}
\pgfpathcurveto{\pgfqpoint{1.28cm}{1.887cm}}{\pgfqpoint{1.245cm}{1.902cm}}{\pgfqpoint{1.209cm}{1.902cm}}
\pgfpathcurveto{\pgfqpoint{1.172cm}{1.902cm}}{\pgfqpoint{1.138cm}{1.887cm}}{\pgfqpoint{1.112cm}{1.862cm}}
\pgfpathcurveto{\pgfqpoint{1.087cm}{1.836cm}}{\pgfqpoint{1.072cm}{1.801cm}}{\pgfqpoint{1.072cm}{1.765cm}}
\pgfpathcurveto{\pgfqpoint{1.072cm}{1.728cm}}{\pgfqpoint{1.087cm}{1.694cm}}{\pgfqpoint{1.112cm}{1.668cm}}
\pgfpathcurveto{\pgfqpoint{1.138cm}{1.642cm}}{\pgfqpoint{1.172cm}{1.628cm}}{\pgfqpoint{1.209cm}{1.628cm}}
\pgfpathcurveto{\pgfqpoint{1.245cm}{1.628cm}}{\pgfqpoint{1.28cm}{1.642cm}}{\pgfqpoint{1.305cm}{1.668cm}}
\pgfpathcurveto{\pgfqpoint{1.331cm}{1.694cm}}{\pgfqpoint{1.345cm}{1.728cm}}{\pgfqpoint{1.345cm}{1.765cm}}
\pgfusepath{fill}
\begin{pgfscope}
\pgfsetdash{}{0cm}
\pgfsetlinewidth{0.818mm}
\pgfsetroundcap
\pgfsetroundjoin
\pgfsetmiterlimit{7.0}
\pgfpathmoveto{\pgfqpoint{0.682cm}{1.065cm}}
\pgfpathlineto{\pgfqpoint{1.246cm}{0.315cm}}
\pgfpathlineto{\pgfqpoint{1.811cm}{1.065cm}}
\pgfusepath{stroke}
\end{pgfscope}
\pgfpathmoveto{\pgfqpoint{1.948cm}{1.065cm}}
\pgfpathcurveto{\pgfqpoint{1.948cm}{1.101cm}}{\pgfqpoint{1.933cm}{1.136cm}}{\pgfqpoint{1.907cm}{1.162cm}}
\pgfpathcurveto{\pgfqpoint{1.882cm}{1.187cm}}{\pgfqpoint{1.847cm}{1.202cm}}{\pgfqpoint{1.811cm}{1.202cm}}
\pgfpathcurveto{\pgfqpoint{1.775cm}{1.202cm}}{\pgfqpoint{1.74cm}{1.187cm}}{\pgfqpoint{1.714cm}{1.162cm}}
\pgfpathcurveto{\pgfqpoint{1.689cm}{1.136cm}}{\pgfqpoint{1.674cm}{1.101cm}}{\pgfqpoint{1.674cm}{1.065cm}}
\pgfpathcurveto{\pgfqpoint{1.674cm}{1.029cm}}{\pgfqpoint{1.689cm}{0.994cm}}{\pgfqpoint{1.714cm}{0.968cm}}
\pgfpathcurveto{\pgfqpoint{1.74cm}{0.942cm}}{\pgfqpoint{1.775cm}{0.928cm}}{\pgfqpoint{1.811cm}{0.928cm}}
\pgfpathcurveto{\pgfqpoint{1.847cm}{0.928cm}}{\pgfqpoint{1.882cm}{0.942cm}}{\pgfqpoint{1.907cm}{0.968cm}}
\pgfpathcurveto{\pgfqpoint{1.933cm}{0.994cm}}{\pgfqpoint{1.948cm}{1.029cm}}{\pgfqpoint{1.948cm}{1.065cm}}
\pgfusepath{fill}
\begin{pgfscope}
\pgfsetdash{}{0cm}
\pgfsetlinewidth{0.818mm}
\pgfsetmiterlimit{4.0}
\pgfpathmoveto{\pgfqpoint{1.383cm}{0.178cm}}
\pgfpathcurveto{\pgfqpoint{1.383cm}{0.214cm}}{\pgfqpoint{1.369cm}{0.249cm}}{\pgfqpoint{1.343cm}{0.275cm}}
\pgfpathcurveto{\pgfqpoint{1.317cm}{0.3cm}}{\pgfqpoint{1.283cm}{0.315cm}}{\pgfqpoint{1.246cm}{0.315cm}}
\pgfpathcurveto{\pgfqpoint{1.21cm}{0.315cm}}{\pgfqpoint{1.175cm}{0.3cm}}{\pgfqpoint{1.15cm}{0.275cm}}
\pgfpathcurveto{\pgfqpoint{1.124cm}{0.249cm}}{\pgfqpoint{1.11cm}{0.214cm}}{\pgfqpoint{1.11cm}{0.178cm}}
\pgfpathcurveto{\pgfqpoint{1.11cm}{0.141cm}}{\pgfqpoint{1.124cm}{0.107cm}}{\pgfqpoint{1.15cm}{0.081cm}}
\pgfpathcurveto{\pgfqpoint{1.175cm}{0.055cm}}{\pgfqpoint{1.21cm}{0.041cm}}{\pgfqpoint{1.246cm}{0.041cm}}
\pgfpathcurveto{\pgfqpoint{1.283cm}{0.041cm}}{\pgfqpoint{1.317cm}{0.055cm}}{\pgfqpoint{1.343cm}{0.081cm}}
\pgfpathcurveto{\pgfqpoint{1.369cm}{0.107cm}}{\pgfqpoint{1.383cm}{0.141cm}}{\pgfqpoint{1.383cm}{0.178cm}}
\pgfusepath{stroke}
\end{pgfscope}
\end{pgfscope}
\end{pgfscope}
\end{pgfscope}
\end{tikzpicture}}}+6\mathrm{com}(X^{\!\resizebox{0.6em}{!}{
\begin{tikzpicture}
\pgfpathmoveto{\pgfqpoint{0cm}{-0.035cm}}
\pgfpathlineto{\pgfqpoint{1.376cm}{-0.035cm}}
\pgfpathlineto{\pgfqpoint{1.376cm}{1.552cm}}
\pgfpathlineto{\pgfqpoint{0cm}{1.552cm}}
\pgfpathclose
\pgfusepath{clip}
\begin{pgfscope}
\begin{pgfscope}
\pgfpathmoveto{\pgfqpoint{0cm}{-0.035cm}}
\pgfpathlineto{\pgfqpoint{1.376cm}{-0.035cm}}
\pgfpathlineto{\pgfqpoint{1.376cm}{1.552cm}}
\pgfpathlineto{\pgfqpoint{0cm}{1.552cm}}
\pgfpathclose
\pgfusepath{clip}
\begin{pgfscope}
\begin{pgfscope}
\pgfsetdash{}{0cm}
\pgfsetlinewidth{0.818mm}
\pgfsetroundcap
\pgfsetroundjoin
\pgfsetmiterlimit{7.0}
\definecolor{eps2pgf_color}{gray}{0}\pgfsetstrokecolor{eps2pgf_color}\pgfsetfillcolor{eps2pgf_color}
\pgfpathmoveto{\pgfqpoint{0.117cm}{1.421cm}}
\pgfpathlineto{\pgfqpoint{0.682cm}{0.671cm}}
\pgfpathlineto{\pgfqpoint{1.246cm}{1.421cm}}
\pgfusepath{stroke}
\end{pgfscope}
\definecolor{eps2pgf_color}{gray}{0}\pgfsetstrokecolor{eps2pgf_color}\pgfsetfillcolor{eps2pgf_color}
\pgfpathmoveto{\pgfqpoint{0.273cm}{1.395cm}}
\pgfpathcurveto{\pgfqpoint{0.273cm}{1.432cm}}{\pgfqpoint{0.259cm}{1.467cm}}{\pgfqpoint{0.233cm}{1.492cm}}
\pgfpathcurveto{\pgfqpoint{0.207cm}{1.518cm}}{\pgfqpoint{0.173cm}{1.532cm}}{\pgfqpoint{0.137cm}{1.532cm}}
\pgfpathcurveto{\pgfqpoint{0.1cm}{1.532cm}}{\pgfqpoint{0.066cm}{1.518cm}}{\pgfqpoint{0.04cm}{1.492cm}}
\pgfpathcurveto{\pgfqpoint{0.014cm}{1.467cm}}{\pgfqpoint{0cm}{1.432cm}}{\pgfqpoint{0cm}{1.395cm}}
\pgfpathcurveto{\pgfqpoint{0cm}{1.359cm}}{\pgfqpoint{0.014cm}{1.324cm}}{\pgfqpoint{0.04cm}{1.299cm}}
\pgfpathcurveto{\pgfqpoint{0.066cm}{1.273cm}}{\pgfqpoint{0.1cm}{1.258cm}}{\pgfqpoint{0.137cm}{1.258cm}}
\pgfpathcurveto{\pgfqpoint{0.173cm}{1.258cm}}{\pgfqpoint{0.207cm}{1.273cm}}{\pgfqpoint{0.233cm}{1.299cm}}
\pgfpathcurveto{\pgfqpoint{0.259cm}{1.324cm}}{\pgfqpoint{0.273cm}{1.359cm}}{\pgfqpoint{0.273cm}{1.395cm}}
\pgfusepath{fill}
\begin{pgfscope}
\pgfsetdash{}{0cm}
\pgfsetlinewidth{0.818mm}
\pgfsetmiterlimit{7.0}
\pgfpathmoveto{\pgfqpoint{0.682cm}{0.671cm}}
\pgfpathlineto{\pgfqpoint{0.679cm}{1.418cm}}
\pgfusepath{stroke}
\end{pgfscope}
\pgfpathmoveto{\pgfqpoint{0.815cm}{1.399cm}}
\pgfpathcurveto{\pgfqpoint{0.815cm}{1.435cm}}{\pgfqpoint{0.801cm}{1.47cm}}{\pgfqpoint{0.775cm}{1.496cm}}
\pgfpathcurveto{\pgfqpoint{0.75cm}{1.521cm}}{\pgfqpoint{0.715cm}{1.536cm}}{\pgfqpoint{0.679cm}{1.536cm}}
\pgfpathcurveto{\pgfqpoint{0.643cm}{1.536cm}}{\pgfqpoint{0.608cm}{1.521cm}}{\pgfqpoint{0.582cm}{1.496cm}}
\pgfpathcurveto{\pgfqpoint{0.557cm}{1.47cm}}{\pgfqpoint{0.542cm}{1.435cm}}{\pgfqpoint{0.542cm}{1.399cm}}
\pgfpathcurveto{\pgfqpoint{0.542cm}{1.363cm}}{\pgfqpoint{0.557cm}{1.328cm}}{\pgfqpoint{0.582cm}{1.302cm}}
\pgfpathcurveto{\pgfqpoint{0.608cm}{1.276cm}}{\pgfqpoint{0.643cm}{1.262cm}}{\pgfqpoint{0.679cm}{1.262cm}}
\pgfpathcurveto{\pgfqpoint{0.715cm}{1.262cm}}{\pgfqpoint{0.75cm}{1.276cm}}{\pgfqpoint{0.775cm}{1.302cm}}
\pgfpathcurveto{\pgfqpoint{0.801cm}{1.328cm}}{\pgfqpoint{0.815cm}{1.363cm}}{\pgfqpoint{0.815cm}{1.399cm}}
\pgfusepath{fill}
\pgfpathmoveto{\pgfqpoint{1.345cm}{1.371cm}}
\pgfpathcurveto{\pgfqpoint{1.345cm}{1.408cm}}{\pgfqpoint{1.331cm}{1.442cm}}{\pgfqpoint{1.305cm}{1.468cm}}
\pgfpathcurveto{\pgfqpoint{1.28cm}{1.494cm}}{\pgfqpoint{1.245cm}{1.508cm}}{\pgfqpoint{1.209cm}{1.508cm}}
\pgfpathcurveto{\pgfqpoint{1.172cm}{1.508cm}}{\pgfqpoint{1.138cm}{1.494cm}}{\pgfqpoint{1.112cm}{1.468cm}}
\pgfpathcurveto{\pgfqpoint{1.087cm}{1.442cm}}{\pgfqpoint{1.072cm}{1.408cm}}{\pgfqpoint{1.072cm}{1.371cm}}
\pgfpathcurveto{\pgfqpoint{1.072cm}{1.335cm}}{\pgfqpoint{1.087cm}{1.3cm}}{\pgfqpoint{1.112cm}{1.274cm}}
\pgfpathcurveto{\pgfqpoint{1.138cm}{1.249cm}}{\pgfqpoint{1.172cm}{1.234cm}}{\pgfqpoint{1.209cm}{1.234cm}}
\pgfpathcurveto{\pgfqpoint{1.245cm}{1.234cm}}{\pgfqpoint{1.28cm}{1.249cm}}{\pgfqpoint{1.305cm}{1.274cm}}
\pgfpathcurveto{\pgfqpoint{1.331cm}{1.3cm}}{\pgfqpoint{1.345cm}{1.335cm}}{\pgfqpoint{1.345cm}{1.371cm}}
\pgfusepath{fill}
\begin{pgfscope}
\pgfsetdash{}{0cm}
\pgfsetlinewidth{0.818mm}
\pgfsetroundcap
\pgfsetmiterlimit{4.0}
\pgfpathmoveto{\pgfqpoint{0.682cm}{0.671cm}}
\pgfpathlineto{\pgfqpoint{0.682cm}{0.042cm}}
\pgfusepath{stroke}
\end{pgfscope}
\end{pgfscope}
\end{pgfscope}
\end{pgfscope}
\end{tikzpicture}}},X^{\!\resizebox{0.6em}{!}{
\begin{tikzpicture}
\pgfpathmoveto{\pgfqpoint{0cm}{-0.035cm}}
\pgfpathlineto{\pgfqpoint{1.376cm}{-0.035cm}}
\pgfpathlineto{\pgfqpoint{1.376cm}{1.552cm}}
\pgfpathlineto{\pgfqpoint{0cm}{1.552cm}}
\pgfpathclose
\pgfusepath{clip}
\begin{pgfscope}
\begin{pgfscope}
\pgfpathmoveto{\pgfqpoint{0cm}{-0.035cm}}
\pgfpathlineto{\pgfqpoint{1.376cm}{-0.035cm}}
\pgfpathlineto{\pgfqpoint{1.376cm}{1.552cm}}
\pgfpathlineto{\pgfqpoint{0cm}{1.552cm}}
\pgfpathclose
\pgfusepath{clip}
\begin{pgfscope}
\begin{pgfscope}
\pgfsetdash{}{0cm}
\pgfsetlinewidth{0.818mm}
\pgfsetroundcap
\pgfsetroundjoin
\pgfsetmiterlimit{7.0}
\definecolor{eps2pgf_color}{gray}{0}\pgfsetstrokecolor{eps2pgf_color}\pgfsetfillcolor{eps2pgf_color}
\pgfpathmoveto{\pgfqpoint{0.117cm}{1.421cm}}
\pgfpathlineto{\pgfqpoint{0.682cm}{0.671cm}}
\pgfpathlineto{\pgfqpoint{1.246cm}{1.421cm}}
\pgfusepath{stroke}
\end{pgfscope}
\definecolor{eps2pgf_color}{gray}{0}\pgfsetstrokecolor{eps2pgf_color}\pgfsetfillcolor{eps2pgf_color}
\pgfpathmoveto{\pgfqpoint{0.273cm}{1.395cm}}
\pgfpathcurveto{\pgfqpoint{0.273cm}{1.432cm}}{\pgfqpoint{0.259cm}{1.467cm}}{\pgfqpoint{0.233cm}{1.492cm}}
\pgfpathcurveto{\pgfqpoint{0.207cm}{1.518cm}}{\pgfqpoint{0.173cm}{1.532cm}}{\pgfqpoint{0.137cm}{1.532cm}}
\pgfpathcurveto{\pgfqpoint{0.1cm}{1.532cm}}{\pgfqpoint{0.066cm}{1.518cm}}{\pgfqpoint{0.04cm}{1.492cm}}
\pgfpathcurveto{\pgfqpoint{0.014cm}{1.467cm}}{\pgfqpoint{0cm}{1.432cm}}{\pgfqpoint{0cm}{1.395cm}}
\pgfpathcurveto{\pgfqpoint{0cm}{1.359cm}}{\pgfqpoint{0.014cm}{1.324cm}}{\pgfqpoint{0.04cm}{1.299cm}}
\pgfpathcurveto{\pgfqpoint{0.066cm}{1.273cm}}{\pgfqpoint{0.1cm}{1.258cm}}{\pgfqpoint{0.137cm}{1.258cm}}
\pgfpathcurveto{\pgfqpoint{0.173cm}{1.258cm}}{\pgfqpoint{0.207cm}{1.273cm}}{\pgfqpoint{0.233cm}{1.299cm}}
\pgfpathcurveto{\pgfqpoint{0.259cm}{1.324cm}}{\pgfqpoint{0.273cm}{1.359cm}}{\pgfqpoint{0.273cm}{1.395cm}}
\pgfusepath{fill}
\begin{pgfscope}
\pgfsetdash{}{0cm}
\pgfsetlinewidth{0.818mm}
\pgfsetmiterlimit{7.0}
\pgfpathmoveto{\pgfqpoint{0.682cm}{0.671cm}}
\pgfpathlineto{\pgfqpoint{0.679cm}{1.418cm}}
\pgfusepath{stroke}
\end{pgfscope}
\pgfpathmoveto{\pgfqpoint{0.815cm}{1.399cm}}
\pgfpathcurveto{\pgfqpoint{0.815cm}{1.435cm}}{\pgfqpoint{0.801cm}{1.47cm}}{\pgfqpoint{0.775cm}{1.496cm}}
\pgfpathcurveto{\pgfqpoint{0.75cm}{1.521cm}}{\pgfqpoint{0.715cm}{1.536cm}}{\pgfqpoint{0.679cm}{1.536cm}}
\pgfpathcurveto{\pgfqpoint{0.643cm}{1.536cm}}{\pgfqpoint{0.608cm}{1.521cm}}{\pgfqpoint{0.582cm}{1.496cm}}
\pgfpathcurveto{\pgfqpoint{0.557cm}{1.47cm}}{\pgfqpoint{0.542cm}{1.435cm}}{\pgfqpoint{0.542cm}{1.399cm}}
\pgfpathcurveto{\pgfqpoint{0.542cm}{1.363cm}}{\pgfqpoint{0.557cm}{1.328cm}}{\pgfqpoint{0.582cm}{1.302cm}}
\pgfpathcurveto{\pgfqpoint{0.608cm}{1.276cm}}{\pgfqpoint{0.643cm}{1.262cm}}{\pgfqpoint{0.679cm}{1.262cm}}
\pgfpathcurveto{\pgfqpoint{0.715cm}{1.262cm}}{\pgfqpoint{0.75cm}{1.276cm}}{\pgfqpoint{0.775cm}{1.302cm}}
\pgfpathcurveto{\pgfqpoint{0.801cm}{1.328cm}}{\pgfqpoint{0.815cm}{1.363cm}}{\pgfqpoint{0.815cm}{1.399cm}}
\pgfusepath{fill}
\pgfpathmoveto{\pgfqpoint{1.345cm}{1.371cm}}
\pgfpathcurveto{\pgfqpoint{1.345cm}{1.408cm}}{\pgfqpoint{1.331cm}{1.442cm}}{\pgfqpoint{1.305cm}{1.468cm}}
\pgfpathcurveto{\pgfqpoint{1.28cm}{1.494cm}}{\pgfqpoint{1.245cm}{1.508cm}}{\pgfqpoint{1.209cm}{1.508cm}}
\pgfpathcurveto{\pgfqpoint{1.172cm}{1.508cm}}{\pgfqpoint{1.138cm}{1.494cm}}{\pgfqpoint{1.112cm}{1.468cm}}
\pgfpathcurveto{\pgfqpoint{1.087cm}{1.442cm}}{\pgfqpoint{1.072cm}{1.408cm}}{\pgfqpoint{1.072cm}{1.371cm}}
\pgfpathcurveto{\pgfqpoint{1.072cm}{1.335cm}}{\pgfqpoint{1.087cm}{1.3cm}}{\pgfqpoint{1.112cm}{1.274cm}}
\pgfpathcurveto{\pgfqpoint{1.138cm}{1.249cm}}{\pgfqpoint{1.172cm}{1.234cm}}{\pgfqpoint{1.209cm}{1.234cm}}
\pgfpathcurveto{\pgfqpoint{1.245cm}{1.234cm}}{\pgfqpoint{1.28cm}{1.249cm}}{\pgfqpoint{1.305cm}{1.274cm}}
\pgfpathcurveto{\pgfqpoint{1.331cm}{1.3cm}}{\pgfqpoint{1.345cm}{1.335cm}}{\pgfqpoint{1.345cm}{1.371cm}}
\pgfusepath{fill}
\begin{pgfscope}
\pgfsetdash{}{0cm}
\pgfsetlinewidth{0.818mm}
\pgfsetroundcap
\pgfsetmiterlimit{4.0}
\pgfpathmoveto{\pgfqpoint{0.682cm}{0.671cm}}
\pgfpathlineto{\pgfqpoint{0.682cm}{0.042cm}}
\pgfusepath{stroke}
\end{pgfscope}
\end{pgfscope}
\end{pgfscope}
\end{pgfscope}
\end{tikzpicture}}},X)+3X\circ (X^{\!\resizebox{0.6em}{!}{
\begin{tikzpicture}
\pgfpathmoveto{\pgfqpoint{0cm}{-0.035cm}}
\pgfpathlineto{\pgfqpoint{1.376cm}{-0.035cm}}
\pgfpathlineto{\pgfqpoint{1.376cm}{1.552cm}}
\pgfpathlineto{\pgfqpoint{0cm}{1.552cm}}
\pgfpathclose
\pgfusepath{clip}
\begin{pgfscope}
\begin{pgfscope}
\pgfpathmoveto{\pgfqpoint{0cm}{-0.035cm}}
\pgfpathlineto{\pgfqpoint{1.376cm}{-0.035cm}}
\pgfpathlineto{\pgfqpoint{1.376cm}{1.552cm}}
\pgfpathlineto{\pgfqpoint{0cm}{1.552cm}}
\pgfpathclose
\pgfusepath{clip}
\begin{pgfscope}
\begin{pgfscope}
\pgfsetdash{}{0cm}
\pgfsetlinewidth{0.818mm}
\pgfsetroundcap
\pgfsetroundjoin
\pgfsetmiterlimit{7.0}
\definecolor{eps2pgf_color}{gray}{0}\pgfsetstrokecolor{eps2pgf_color}\pgfsetfillcolor{eps2pgf_color}
\pgfpathmoveto{\pgfqpoint{0.117cm}{1.421cm}}
\pgfpathlineto{\pgfqpoint{0.682cm}{0.671cm}}
\pgfpathlineto{\pgfqpoint{1.246cm}{1.421cm}}
\pgfusepath{stroke}
\end{pgfscope}
\definecolor{eps2pgf_color}{gray}{0}\pgfsetstrokecolor{eps2pgf_color}\pgfsetfillcolor{eps2pgf_color}
\pgfpathmoveto{\pgfqpoint{0.273cm}{1.395cm}}
\pgfpathcurveto{\pgfqpoint{0.273cm}{1.432cm}}{\pgfqpoint{0.259cm}{1.467cm}}{\pgfqpoint{0.233cm}{1.492cm}}
\pgfpathcurveto{\pgfqpoint{0.207cm}{1.518cm}}{\pgfqpoint{0.173cm}{1.532cm}}{\pgfqpoint{0.137cm}{1.532cm}}
\pgfpathcurveto{\pgfqpoint{0.1cm}{1.532cm}}{\pgfqpoint{0.066cm}{1.518cm}}{\pgfqpoint{0.04cm}{1.492cm}}
\pgfpathcurveto{\pgfqpoint{0.014cm}{1.467cm}}{\pgfqpoint{0cm}{1.432cm}}{\pgfqpoint{0cm}{1.395cm}}
\pgfpathcurveto{\pgfqpoint{0cm}{1.359cm}}{\pgfqpoint{0.014cm}{1.324cm}}{\pgfqpoint{0.04cm}{1.299cm}}
\pgfpathcurveto{\pgfqpoint{0.066cm}{1.273cm}}{\pgfqpoint{0.1cm}{1.258cm}}{\pgfqpoint{0.137cm}{1.258cm}}
\pgfpathcurveto{\pgfqpoint{0.173cm}{1.258cm}}{\pgfqpoint{0.207cm}{1.273cm}}{\pgfqpoint{0.233cm}{1.299cm}}
\pgfpathcurveto{\pgfqpoint{0.259cm}{1.324cm}}{\pgfqpoint{0.273cm}{1.359cm}}{\pgfqpoint{0.273cm}{1.395cm}}
\pgfusepath{fill}
\begin{pgfscope}
\pgfsetdash{}{0cm}
\pgfsetlinewidth{0.818mm}
\pgfsetmiterlimit{7.0}
\pgfpathmoveto{\pgfqpoint{0.682cm}{0.671cm}}
\pgfpathlineto{\pgfqpoint{0.679cm}{1.418cm}}
\pgfusepath{stroke}
\end{pgfscope}
\pgfpathmoveto{\pgfqpoint{0.815cm}{1.399cm}}
\pgfpathcurveto{\pgfqpoint{0.815cm}{1.435cm}}{\pgfqpoint{0.801cm}{1.47cm}}{\pgfqpoint{0.775cm}{1.496cm}}
\pgfpathcurveto{\pgfqpoint{0.75cm}{1.521cm}}{\pgfqpoint{0.715cm}{1.536cm}}{\pgfqpoint{0.679cm}{1.536cm}}
\pgfpathcurveto{\pgfqpoint{0.643cm}{1.536cm}}{\pgfqpoint{0.608cm}{1.521cm}}{\pgfqpoint{0.582cm}{1.496cm}}
\pgfpathcurveto{\pgfqpoint{0.557cm}{1.47cm}}{\pgfqpoint{0.542cm}{1.435cm}}{\pgfqpoint{0.542cm}{1.399cm}}
\pgfpathcurveto{\pgfqpoint{0.542cm}{1.363cm}}{\pgfqpoint{0.557cm}{1.328cm}}{\pgfqpoint{0.582cm}{1.302cm}}
\pgfpathcurveto{\pgfqpoint{0.608cm}{1.276cm}}{\pgfqpoint{0.643cm}{1.262cm}}{\pgfqpoint{0.679cm}{1.262cm}}
\pgfpathcurveto{\pgfqpoint{0.715cm}{1.262cm}}{\pgfqpoint{0.75cm}{1.276cm}}{\pgfqpoint{0.775cm}{1.302cm}}
\pgfpathcurveto{\pgfqpoint{0.801cm}{1.328cm}}{\pgfqpoint{0.815cm}{1.363cm}}{\pgfqpoint{0.815cm}{1.399cm}}
\pgfusepath{fill}
\pgfpathmoveto{\pgfqpoint{1.345cm}{1.371cm}}
\pgfpathcurveto{\pgfqpoint{1.345cm}{1.408cm}}{\pgfqpoint{1.331cm}{1.442cm}}{\pgfqpoint{1.305cm}{1.468cm}}
\pgfpathcurveto{\pgfqpoint{1.28cm}{1.494cm}}{\pgfqpoint{1.245cm}{1.508cm}}{\pgfqpoint{1.209cm}{1.508cm}}
\pgfpathcurveto{\pgfqpoint{1.172cm}{1.508cm}}{\pgfqpoint{1.138cm}{1.494cm}}{\pgfqpoint{1.112cm}{1.468cm}}
\pgfpathcurveto{\pgfqpoint{1.087cm}{1.442cm}}{\pgfqpoint{1.072cm}{1.408cm}}{\pgfqpoint{1.072cm}{1.371cm}}
\pgfpathcurveto{\pgfqpoint{1.072cm}{1.335cm}}{\pgfqpoint{1.087cm}{1.3cm}}{\pgfqpoint{1.112cm}{1.274cm}}
\pgfpathcurveto{\pgfqpoint{1.138cm}{1.249cm}}{\pgfqpoint{1.172cm}{1.234cm}}{\pgfqpoint{1.209cm}{1.234cm}}
\pgfpathcurveto{\pgfqpoint{1.245cm}{1.234cm}}{\pgfqpoint{1.28cm}{1.249cm}}{\pgfqpoint{1.305cm}{1.274cm}}
\pgfpathcurveto{\pgfqpoint{1.331cm}{1.3cm}}{\pgfqpoint{1.345cm}{1.335cm}}{\pgfqpoint{1.345cm}{1.371cm}}
\pgfusepath{fill}
\begin{pgfscope}
\pgfsetdash{}{0cm}
\pgfsetlinewidth{0.818mm}
\pgfsetroundcap
\pgfsetmiterlimit{4.0}
\pgfpathmoveto{\pgfqpoint{0.682cm}{0.671cm}}
\pgfpathlineto{\pgfqpoint{0.682cm}{0.042cm}}
\pgfusepath{stroke}
\end{pgfscope}
\end{pgfscope}
\end{pgfscope}
\end{pgfscope}
\end{tikzpicture}}}\circ X^{\!\resizebox{0.6em}{!}{
\begin{tikzpicture}
\pgfpathmoveto{\pgfqpoint{0cm}{-0.035cm}}
\pgfpathlineto{\pgfqpoint{1.376cm}{-0.035cm}}
\pgfpathlineto{\pgfqpoint{1.376cm}{1.552cm}}
\pgfpathlineto{\pgfqpoint{0cm}{1.552cm}}
\pgfpathclose
\pgfusepath{clip}
\begin{pgfscope}
\begin{pgfscope}
\pgfpathmoveto{\pgfqpoint{0cm}{-0.035cm}}
\pgfpathlineto{\pgfqpoint{1.376cm}{-0.035cm}}
\pgfpathlineto{\pgfqpoint{1.376cm}{1.552cm}}
\pgfpathlineto{\pgfqpoint{0cm}{1.552cm}}
\pgfpathclose
\pgfusepath{clip}
\begin{pgfscope}
\begin{pgfscope}
\pgfsetdash{}{0cm}
\pgfsetlinewidth{0.818mm}
\pgfsetroundcap
\pgfsetroundjoin
\pgfsetmiterlimit{7.0}
\definecolor{eps2pgf_color}{gray}{0}\pgfsetstrokecolor{eps2pgf_color}\pgfsetfillcolor{eps2pgf_color}
\pgfpathmoveto{\pgfqpoint{0.117cm}{1.421cm}}
\pgfpathlineto{\pgfqpoint{0.682cm}{0.671cm}}
\pgfpathlineto{\pgfqpoint{1.246cm}{1.421cm}}
\pgfusepath{stroke}
\end{pgfscope}
\definecolor{eps2pgf_color}{gray}{0}\pgfsetstrokecolor{eps2pgf_color}\pgfsetfillcolor{eps2pgf_color}
\pgfpathmoveto{\pgfqpoint{0.273cm}{1.395cm}}
\pgfpathcurveto{\pgfqpoint{0.273cm}{1.432cm}}{\pgfqpoint{0.259cm}{1.467cm}}{\pgfqpoint{0.233cm}{1.492cm}}
\pgfpathcurveto{\pgfqpoint{0.207cm}{1.518cm}}{\pgfqpoint{0.173cm}{1.532cm}}{\pgfqpoint{0.137cm}{1.532cm}}
\pgfpathcurveto{\pgfqpoint{0.1cm}{1.532cm}}{\pgfqpoint{0.066cm}{1.518cm}}{\pgfqpoint{0.04cm}{1.492cm}}
\pgfpathcurveto{\pgfqpoint{0.014cm}{1.467cm}}{\pgfqpoint{0cm}{1.432cm}}{\pgfqpoint{0cm}{1.395cm}}
\pgfpathcurveto{\pgfqpoint{0cm}{1.359cm}}{\pgfqpoint{0.014cm}{1.324cm}}{\pgfqpoint{0.04cm}{1.299cm}}
\pgfpathcurveto{\pgfqpoint{0.066cm}{1.273cm}}{\pgfqpoint{0.1cm}{1.258cm}}{\pgfqpoint{0.137cm}{1.258cm}}
\pgfpathcurveto{\pgfqpoint{0.173cm}{1.258cm}}{\pgfqpoint{0.207cm}{1.273cm}}{\pgfqpoint{0.233cm}{1.299cm}}
\pgfpathcurveto{\pgfqpoint{0.259cm}{1.324cm}}{\pgfqpoint{0.273cm}{1.359cm}}{\pgfqpoint{0.273cm}{1.395cm}}
\pgfusepath{fill}
\begin{pgfscope}
\pgfsetdash{}{0cm}
\pgfsetlinewidth{0.818mm}
\pgfsetmiterlimit{7.0}
\pgfpathmoveto{\pgfqpoint{0.682cm}{0.671cm}}
\pgfpathlineto{\pgfqpoint{0.679cm}{1.418cm}}
\pgfusepath{stroke}
\end{pgfscope}
\pgfpathmoveto{\pgfqpoint{0.815cm}{1.399cm}}
\pgfpathcurveto{\pgfqpoint{0.815cm}{1.435cm}}{\pgfqpoint{0.801cm}{1.47cm}}{\pgfqpoint{0.775cm}{1.496cm}}
\pgfpathcurveto{\pgfqpoint{0.75cm}{1.521cm}}{\pgfqpoint{0.715cm}{1.536cm}}{\pgfqpoint{0.679cm}{1.536cm}}
\pgfpathcurveto{\pgfqpoint{0.643cm}{1.536cm}}{\pgfqpoint{0.608cm}{1.521cm}}{\pgfqpoint{0.582cm}{1.496cm}}
\pgfpathcurveto{\pgfqpoint{0.557cm}{1.47cm}}{\pgfqpoint{0.542cm}{1.435cm}}{\pgfqpoint{0.542cm}{1.399cm}}
\pgfpathcurveto{\pgfqpoint{0.542cm}{1.363cm}}{\pgfqpoint{0.557cm}{1.328cm}}{\pgfqpoint{0.582cm}{1.302cm}}
\pgfpathcurveto{\pgfqpoint{0.608cm}{1.276cm}}{\pgfqpoint{0.643cm}{1.262cm}}{\pgfqpoint{0.679cm}{1.262cm}}
\pgfpathcurveto{\pgfqpoint{0.715cm}{1.262cm}}{\pgfqpoint{0.75cm}{1.276cm}}{\pgfqpoint{0.775cm}{1.302cm}}
\pgfpathcurveto{\pgfqpoint{0.801cm}{1.328cm}}{\pgfqpoint{0.815cm}{1.363cm}}{\pgfqpoint{0.815cm}{1.399cm}}
\pgfusepath{fill}
\pgfpathmoveto{\pgfqpoint{1.345cm}{1.371cm}}
\pgfpathcurveto{\pgfqpoint{1.345cm}{1.408cm}}{\pgfqpoint{1.331cm}{1.442cm}}{\pgfqpoint{1.305cm}{1.468cm}}
\pgfpathcurveto{\pgfqpoint{1.28cm}{1.494cm}}{\pgfqpoint{1.245cm}{1.508cm}}{\pgfqpoint{1.209cm}{1.508cm}}
\pgfpathcurveto{\pgfqpoint{1.172cm}{1.508cm}}{\pgfqpoint{1.138cm}{1.494cm}}{\pgfqpoint{1.112cm}{1.468cm}}
\pgfpathcurveto{\pgfqpoint{1.087cm}{1.442cm}}{\pgfqpoint{1.072cm}{1.408cm}}{\pgfqpoint{1.072cm}{1.371cm}}
\pgfpathcurveto{\pgfqpoint{1.072cm}{1.335cm}}{\pgfqpoint{1.087cm}{1.3cm}}{\pgfqpoint{1.112cm}{1.274cm}}
\pgfpathcurveto{\pgfqpoint{1.138cm}{1.249cm}}{\pgfqpoint{1.172cm}{1.234cm}}{\pgfqpoint{1.209cm}{1.234cm}}
\pgfpathcurveto{\pgfqpoint{1.245cm}{1.234cm}}{\pgfqpoint{1.28cm}{1.249cm}}{\pgfqpoint{1.305cm}{1.274cm}}
\pgfpathcurveto{\pgfqpoint{1.331cm}{1.3cm}}{\pgfqpoint{1.345cm}{1.335cm}}{\pgfqpoint{1.345cm}{1.371cm}}
\pgfusepath{fill}
\begin{pgfscope}
\pgfsetdash{}{0cm}
\pgfsetlinewidth{0.818mm}
\pgfsetroundcap
\pgfsetmiterlimit{4.0}
\pgfpathmoveto{\pgfqpoint{0.682cm}{0.671cm}}
\pgfpathlineto{\pgfqpoint{0.682cm}{0.042cm}}
\pgfusepath{stroke}
\end{pgfscope}
\end{pgfscope}
\end{pgfscope}
\end{pgfscope}
\end{tikzpicture}}}).
$$
Similarly we decompose
$$
6X\circ(X^{\!\resizebox{0.6em}{!}{
\begin{tikzpicture}
\pgfpathmoveto{\pgfqpoint{0cm}{-0.035cm}}
\pgfpathlineto{\pgfqpoint{1.376cm}{-0.035cm}}
\pgfpathlineto{\pgfqpoint{1.376cm}{1.552cm}}
\pgfpathlineto{\pgfqpoint{0cm}{1.552cm}}
\pgfpathclose
\pgfusepath{clip}
\begin{pgfscope}
\begin{pgfscope}
\pgfpathmoveto{\pgfqpoint{0cm}{-0.035cm}}
\pgfpathlineto{\pgfqpoint{1.376cm}{-0.035cm}}
\pgfpathlineto{\pgfqpoint{1.376cm}{1.552cm}}
\pgfpathlineto{\pgfqpoint{0cm}{1.552cm}}
\pgfpathclose
\pgfusepath{clip}
\begin{pgfscope}
\begin{pgfscope}
\pgfsetdash{}{0cm}
\pgfsetlinewidth{0.818mm}
\pgfsetroundcap
\pgfsetroundjoin
\pgfsetmiterlimit{7.0}
\definecolor{eps2pgf_color}{gray}{0}\pgfsetstrokecolor{eps2pgf_color}\pgfsetfillcolor{eps2pgf_color}
\pgfpathmoveto{\pgfqpoint{0.117cm}{1.421cm}}
\pgfpathlineto{\pgfqpoint{0.682cm}{0.671cm}}
\pgfpathlineto{\pgfqpoint{1.246cm}{1.421cm}}
\pgfusepath{stroke}
\end{pgfscope}
\definecolor{eps2pgf_color}{gray}{0}\pgfsetstrokecolor{eps2pgf_color}\pgfsetfillcolor{eps2pgf_color}
\pgfpathmoveto{\pgfqpoint{0.273cm}{1.395cm}}
\pgfpathcurveto{\pgfqpoint{0.273cm}{1.432cm}}{\pgfqpoint{0.259cm}{1.467cm}}{\pgfqpoint{0.233cm}{1.492cm}}
\pgfpathcurveto{\pgfqpoint{0.207cm}{1.518cm}}{\pgfqpoint{0.173cm}{1.532cm}}{\pgfqpoint{0.137cm}{1.532cm}}
\pgfpathcurveto{\pgfqpoint{0.1cm}{1.532cm}}{\pgfqpoint{0.066cm}{1.518cm}}{\pgfqpoint{0.04cm}{1.492cm}}
\pgfpathcurveto{\pgfqpoint{0.014cm}{1.467cm}}{\pgfqpoint{0cm}{1.432cm}}{\pgfqpoint{0cm}{1.395cm}}
\pgfpathcurveto{\pgfqpoint{0cm}{1.359cm}}{\pgfqpoint{0.014cm}{1.324cm}}{\pgfqpoint{0.04cm}{1.299cm}}
\pgfpathcurveto{\pgfqpoint{0.066cm}{1.273cm}}{\pgfqpoint{0.1cm}{1.258cm}}{\pgfqpoint{0.137cm}{1.258cm}}
\pgfpathcurveto{\pgfqpoint{0.173cm}{1.258cm}}{\pgfqpoint{0.207cm}{1.273cm}}{\pgfqpoint{0.233cm}{1.299cm}}
\pgfpathcurveto{\pgfqpoint{0.259cm}{1.324cm}}{\pgfqpoint{0.273cm}{1.359cm}}{\pgfqpoint{0.273cm}{1.395cm}}
\pgfusepath{fill}
\begin{pgfscope}
\pgfsetdash{}{0cm}
\pgfsetlinewidth{0.818mm}
\pgfsetmiterlimit{7.0}
\pgfpathmoveto{\pgfqpoint{0.682cm}{0.671cm}}
\pgfpathlineto{\pgfqpoint{0.679cm}{1.418cm}}
\pgfusepath{stroke}
\end{pgfscope}
\pgfpathmoveto{\pgfqpoint{0.815cm}{1.399cm}}
\pgfpathcurveto{\pgfqpoint{0.815cm}{1.435cm}}{\pgfqpoint{0.801cm}{1.47cm}}{\pgfqpoint{0.775cm}{1.496cm}}
\pgfpathcurveto{\pgfqpoint{0.75cm}{1.521cm}}{\pgfqpoint{0.715cm}{1.536cm}}{\pgfqpoint{0.679cm}{1.536cm}}
\pgfpathcurveto{\pgfqpoint{0.643cm}{1.536cm}}{\pgfqpoint{0.608cm}{1.521cm}}{\pgfqpoint{0.582cm}{1.496cm}}
\pgfpathcurveto{\pgfqpoint{0.557cm}{1.47cm}}{\pgfqpoint{0.542cm}{1.435cm}}{\pgfqpoint{0.542cm}{1.399cm}}
\pgfpathcurveto{\pgfqpoint{0.542cm}{1.363cm}}{\pgfqpoint{0.557cm}{1.328cm}}{\pgfqpoint{0.582cm}{1.302cm}}
\pgfpathcurveto{\pgfqpoint{0.608cm}{1.276cm}}{\pgfqpoint{0.643cm}{1.262cm}}{\pgfqpoint{0.679cm}{1.262cm}}
\pgfpathcurveto{\pgfqpoint{0.715cm}{1.262cm}}{\pgfqpoint{0.75cm}{1.276cm}}{\pgfqpoint{0.775cm}{1.302cm}}
\pgfpathcurveto{\pgfqpoint{0.801cm}{1.328cm}}{\pgfqpoint{0.815cm}{1.363cm}}{\pgfqpoint{0.815cm}{1.399cm}}
\pgfusepath{fill}
\pgfpathmoveto{\pgfqpoint{1.345cm}{1.371cm}}
\pgfpathcurveto{\pgfqpoint{1.345cm}{1.408cm}}{\pgfqpoint{1.331cm}{1.442cm}}{\pgfqpoint{1.305cm}{1.468cm}}
\pgfpathcurveto{\pgfqpoint{1.28cm}{1.494cm}}{\pgfqpoint{1.245cm}{1.508cm}}{\pgfqpoint{1.209cm}{1.508cm}}
\pgfpathcurveto{\pgfqpoint{1.172cm}{1.508cm}}{\pgfqpoint{1.138cm}{1.494cm}}{\pgfqpoint{1.112cm}{1.468cm}}
\pgfpathcurveto{\pgfqpoint{1.087cm}{1.442cm}}{\pgfqpoint{1.072cm}{1.408cm}}{\pgfqpoint{1.072cm}{1.371cm}}
\pgfpathcurveto{\pgfqpoint{1.072cm}{1.335cm}}{\pgfqpoint{1.087cm}{1.3cm}}{\pgfqpoint{1.112cm}{1.274cm}}
\pgfpathcurveto{\pgfqpoint{1.138cm}{1.249cm}}{\pgfqpoint{1.172cm}{1.234cm}}{\pgfqpoint{1.209cm}{1.234cm}}
\pgfpathcurveto{\pgfqpoint{1.245cm}{1.234cm}}{\pgfqpoint{1.28cm}{1.249cm}}{\pgfqpoint{1.305cm}{1.274cm}}
\pgfpathcurveto{\pgfqpoint{1.331cm}{1.3cm}}{\pgfqpoint{1.345cm}{1.335cm}}{\pgfqpoint{1.345cm}{1.371cm}}
\pgfusepath{fill}
\begin{pgfscope}
\pgfsetdash{}{0cm}
\pgfsetlinewidth{0.818mm}
\pgfsetroundcap
\pgfsetmiterlimit{4.0}
\pgfpathmoveto{\pgfqpoint{0.682cm}{0.671cm}}
\pgfpathlineto{\pgfqpoint{0.682cm}{0.042cm}}
\pgfusepath{stroke}
\end{pgfscope}
\end{pgfscope}
\end{pgfscope}
\end{pgfscope}
\end{tikzpicture}}}\succ(\phi+\psi))=6(\phi+\psi)X^{\!\resizebox{!}{.8em}{
\begin{tikzpicture}
\pgfpathmoveto{\pgfqpoint{0cm}{-0.035cm}}
\pgfpathlineto{\pgfqpoint{1.976cm}{-0.035cm}}
\pgfpathlineto{\pgfqpoint{1.976cm}{1.94cm}}
\pgfpathlineto{\pgfqpoint{0cm}{1.94cm}}
\pgfpathclose
\pgfusepath{clip}
\begin{pgfscope}
\begin{pgfscope}
\pgfpathmoveto{\pgfqpoint{0cm}{-0.035cm}}
\pgfpathlineto{\pgfqpoint{1.976cm}{-0.035cm}}
\pgfpathlineto{\pgfqpoint{1.976cm}{1.94cm}}
\pgfpathlineto{\pgfqpoint{0cm}{1.94cm}}
\pgfpathclose
\pgfusepath{clip}
\begin{pgfscope}
\begin{pgfscope}
\pgfsetdash{}{0cm}
\pgfsetlinewidth{0.818mm}
\pgfsetroundcap
\pgfsetroundjoin
\pgfsetmiterlimit{7.0}
\definecolor{eps2pgf_color}{gray}{0}\pgfsetstrokecolor{eps2pgf_color}\pgfsetfillcolor{eps2pgf_color}
\pgfpathmoveto{\pgfqpoint{0.117cm}{1.815cm}}
\pgfpathlineto{\pgfqpoint{0.682cm}{1.065cm}}
\pgfpathlineto{\pgfqpoint{1.246cm}{1.815cm}}
\pgfusepath{stroke}
\end{pgfscope}
\definecolor{eps2pgf_color}{gray}{0}\pgfsetstrokecolor{eps2pgf_color}\pgfsetfillcolor{eps2pgf_color}
\pgfpathmoveto{\pgfqpoint{0.273cm}{1.789cm}}
\pgfpathcurveto{\pgfqpoint{0.273cm}{1.825cm}}{\pgfqpoint{0.259cm}{1.86cm}}{\pgfqpoint{0.233cm}{1.886cm}}
\pgfpathcurveto{\pgfqpoint{0.207cm}{1.912cm}}{\pgfqpoint{0.173cm}{1.926cm}}{\pgfqpoint{0.137cm}{1.926cm}}
\pgfpathcurveto{\pgfqpoint{0.1cm}{1.926cm}}{\pgfqpoint{0.066cm}{1.912cm}}{\pgfqpoint{0.04cm}{1.886cm}}
\pgfpathcurveto{\pgfqpoint{0.014cm}{1.86cm}}{\pgfqpoint{0cm}{1.825cm}}{\pgfqpoint{0cm}{1.789cm}}
\pgfpathcurveto{\pgfqpoint{0cm}{1.753cm}}{\pgfqpoint{0.014cm}{1.718cm}}{\pgfqpoint{0.04cm}{1.692cm}}
\pgfpathcurveto{\pgfqpoint{0.066cm}{1.667cm}}{\pgfqpoint{0.1cm}{1.652cm}}{\pgfqpoint{0.137cm}{1.652cm}}
\pgfpathcurveto{\pgfqpoint{0.173cm}{1.652cm}}{\pgfqpoint{0.207cm}{1.667cm}}{\pgfqpoint{0.233cm}{1.692cm}}
\pgfpathcurveto{\pgfqpoint{0.259cm}{1.718cm}}{\pgfqpoint{0.273cm}{1.753cm}}{\pgfqpoint{0.273cm}{1.789cm}}
\pgfusepath{fill}
\begin{pgfscope}
\pgfsetdash{}{0cm}
\pgfsetlinewidth{0.818mm}
\pgfsetmiterlimit{7.0}
\pgfpathmoveto{\pgfqpoint{0.682cm}{1.065cm}}
\pgfpathlineto{\pgfqpoint{0.679cm}{1.812cm}}
\pgfusepath{stroke}
\end{pgfscope}
\pgfpathmoveto{\pgfqpoint{0.815cm}{1.793cm}}
\pgfpathcurveto{\pgfqpoint{0.815cm}{1.829cm}}{\pgfqpoint{0.801cm}{1.864cm}}{\pgfqpoint{0.775cm}{1.89cm}}
\pgfpathcurveto{\pgfqpoint{0.75cm}{1.915cm}}{\pgfqpoint{0.715cm}{1.93cm}}{\pgfqpoint{0.679cm}{1.93cm}}
\pgfpathcurveto{\pgfqpoint{0.643cm}{1.93cm}}{\pgfqpoint{0.608cm}{1.915cm}}{\pgfqpoint{0.582cm}{1.89cm}}
\pgfpathcurveto{\pgfqpoint{0.557cm}{1.864cm}}{\pgfqpoint{0.542cm}{1.829cm}}{\pgfqpoint{0.542cm}{1.793cm}}
\pgfpathcurveto{\pgfqpoint{0.542cm}{1.756cm}}{\pgfqpoint{0.557cm}{1.722cm}}{\pgfqpoint{0.582cm}{1.696cm}}
\pgfpathcurveto{\pgfqpoint{0.608cm}{1.67cm}}{\pgfqpoint{0.643cm}{1.656cm}}{\pgfqpoint{0.679cm}{1.656cm}}
\pgfpathcurveto{\pgfqpoint{0.715cm}{1.656cm}}{\pgfqpoint{0.75cm}{1.67cm}}{\pgfqpoint{0.775cm}{1.696cm}}
\pgfpathcurveto{\pgfqpoint{0.801cm}{1.722cm}}{\pgfqpoint{0.815cm}{1.756cm}}{\pgfqpoint{0.815cm}{1.793cm}}
\pgfusepath{fill}
\pgfpathmoveto{\pgfqpoint{1.345cm}{1.765cm}}
\pgfpathcurveto{\pgfqpoint{1.345cm}{1.801cm}}{\pgfqpoint{1.331cm}{1.836cm}}{\pgfqpoint{1.305cm}{1.862cm}}
\pgfpathcurveto{\pgfqpoint{1.28cm}{1.887cm}}{\pgfqpoint{1.245cm}{1.902cm}}{\pgfqpoint{1.209cm}{1.902cm}}
\pgfpathcurveto{\pgfqpoint{1.172cm}{1.902cm}}{\pgfqpoint{1.138cm}{1.887cm}}{\pgfqpoint{1.112cm}{1.862cm}}
\pgfpathcurveto{\pgfqpoint{1.087cm}{1.836cm}}{\pgfqpoint{1.072cm}{1.801cm}}{\pgfqpoint{1.072cm}{1.765cm}}
\pgfpathcurveto{\pgfqpoint{1.072cm}{1.728cm}}{\pgfqpoint{1.087cm}{1.694cm}}{\pgfqpoint{1.112cm}{1.668cm}}
\pgfpathcurveto{\pgfqpoint{1.138cm}{1.642cm}}{\pgfqpoint{1.172cm}{1.628cm}}{\pgfqpoint{1.209cm}{1.628cm}}
\pgfpathcurveto{\pgfqpoint{1.245cm}{1.628cm}}{\pgfqpoint{1.28cm}{1.642cm}}{\pgfqpoint{1.305cm}{1.668cm}}
\pgfpathcurveto{\pgfqpoint{1.331cm}{1.694cm}}{\pgfqpoint{1.345cm}{1.728cm}}{\pgfqpoint{1.345cm}{1.765cm}}
\pgfusepath{fill}
\begin{pgfscope}
\pgfsetdash{}{0cm}
\pgfsetlinewidth{0.818mm}
\pgfsetroundcap
\pgfsetroundjoin
\pgfsetmiterlimit{7.0}
\pgfpathmoveto{\pgfqpoint{0.682cm}{1.065cm}}
\pgfpathlineto{\pgfqpoint{1.246cm}{0.315cm}}
\pgfpathlineto{\pgfqpoint{1.811cm}{1.065cm}}
\pgfusepath{stroke}
\end{pgfscope}
\pgfpathmoveto{\pgfqpoint{1.948cm}{1.065cm}}
\pgfpathcurveto{\pgfqpoint{1.948cm}{1.101cm}}{\pgfqpoint{1.933cm}{1.136cm}}{\pgfqpoint{1.907cm}{1.162cm}}
\pgfpathcurveto{\pgfqpoint{1.882cm}{1.187cm}}{\pgfqpoint{1.847cm}{1.202cm}}{\pgfqpoint{1.811cm}{1.202cm}}
\pgfpathcurveto{\pgfqpoint{1.775cm}{1.202cm}}{\pgfqpoint{1.74cm}{1.187cm}}{\pgfqpoint{1.714cm}{1.162cm}}
\pgfpathcurveto{\pgfqpoint{1.689cm}{1.136cm}}{\pgfqpoint{1.674cm}{1.101cm}}{\pgfqpoint{1.674cm}{1.065cm}}
\pgfpathcurveto{\pgfqpoint{1.674cm}{1.029cm}}{\pgfqpoint{1.689cm}{0.994cm}}{\pgfqpoint{1.714cm}{0.968cm}}
\pgfpathcurveto{\pgfqpoint{1.74cm}{0.942cm}}{\pgfqpoint{1.775cm}{0.928cm}}{\pgfqpoint{1.811cm}{0.928cm}}
\pgfpathcurveto{\pgfqpoint{1.847cm}{0.928cm}}{\pgfqpoint{1.882cm}{0.942cm}}{\pgfqpoint{1.907cm}{0.968cm}}
\pgfpathcurveto{\pgfqpoint{1.933cm}{0.994cm}}{\pgfqpoint{1.948cm}{1.029cm}}{\pgfqpoint{1.948cm}{1.065cm}}
\pgfusepath{fill}
\begin{pgfscope}
\pgfsetdash{}{0cm}
\pgfsetlinewidth{0.818mm}
\pgfsetmiterlimit{4.0}
\pgfpathmoveto{\pgfqpoint{1.383cm}{0.178cm}}
\pgfpathcurveto{\pgfqpoint{1.383cm}{0.214cm}}{\pgfqpoint{1.369cm}{0.249cm}}{\pgfqpoint{1.343cm}{0.275cm}}
\pgfpathcurveto{\pgfqpoint{1.317cm}{0.3cm}}{\pgfqpoint{1.283cm}{0.315cm}}{\pgfqpoint{1.246cm}{0.315cm}}
\pgfpathcurveto{\pgfqpoint{1.21cm}{0.315cm}}{\pgfqpoint{1.175cm}{0.3cm}}{\pgfqpoint{1.15cm}{0.275cm}}
\pgfpathcurveto{\pgfqpoint{1.124cm}{0.249cm}}{\pgfqpoint{1.11cm}{0.214cm}}{\pgfqpoint{1.11cm}{0.178cm}}
\pgfpathcurveto{\pgfqpoint{1.11cm}{0.141cm}}{\pgfqpoint{1.124cm}{0.107cm}}{\pgfqpoint{1.15cm}{0.081cm}}
\pgfpathcurveto{\pgfqpoint{1.175cm}{0.055cm}}{\pgfqpoint{1.21cm}{0.041cm}}{\pgfqpoint{1.246cm}{0.041cm}}
\pgfpathcurveto{\pgfqpoint{1.283cm}{0.041cm}}{\pgfqpoint{1.317cm}{0.055cm}}{\pgfqpoint{1.343cm}{0.081cm}}
\pgfpathcurveto{\pgfqpoint{1.369cm}{0.107cm}}{\pgfqpoint{1.383cm}{0.141cm}}{\pgfqpoint{1.383cm}{0.178cm}}
\pgfusepath{stroke}
\end{pgfscope}
\end{pgfscope}
\end{pgfscope}
\end{pgfscope}
\end{tikzpicture}}}+6\mathrm{com}(\phi+\psi,X^{\!\resizebox{0.6em}{!}{
\begin{tikzpicture}
\pgfpathmoveto{\pgfqpoint{0cm}{-0.035cm}}
\pgfpathlineto{\pgfqpoint{1.376cm}{-0.035cm}}
\pgfpathlineto{\pgfqpoint{1.376cm}{1.552cm}}
\pgfpathlineto{\pgfqpoint{0cm}{1.552cm}}
\pgfpathclose
\pgfusepath{clip}
\begin{pgfscope}
\begin{pgfscope}
\pgfpathmoveto{\pgfqpoint{0cm}{-0.035cm}}
\pgfpathlineto{\pgfqpoint{1.376cm}{-0.035cm}}
\pgfpathlineto{\pgfqpoint{1.376cm}{1.552cm}}
\pgfpathlineto{\pgfqpoint{0cm}{1.552cm}}
\pgfpathclose
\pgfusepath{clip}
\begin{pgfscope}
\begin{pgfscope}
\pgfsetdash{}{0cm}
\pgfsetlinewidth{0.818mm}
\pgfsetroundcap
\pgfsetroundjoin
\pgfsetmiterlimit{7.0}
\definecolor{eps2pgf_color}{gray}{0}\pgfsetstrokecolor{eps2pgf_color}\pgfsetfillcolor{eps2pgf_color}
\pgfpathmoveto{\pgfqpoint{0.117cm}{1.421cm}}
\pgfpathlineto{\pgfqpoint{0.682cm}{0.671cm}}
\pgfpathlineto{\pgfqpoint{1.246cm}{1.421cm}}
\pgfusepath{stroke}
\end{pgfscope}
\definecolor{eps2pgf_color}{gray}{0}\pgfsetstrokecolor{eps2pgf_color}\pgfsetfillcolor{eps2pgf_color}
\pgfpathmoveto{\pgfqpoint{0.273cm}{1.395cm}}
\pgfpathcurveto{\pgfqpoint{0.273cm}{1.432cm}}{\pgfqpoint{0.259cm}{1.467cm}}{\pgfqpoint{0.233cm}{1.492cm}}
\pgfpathcurveto{\pgfqpoint{0.207cm}{1.518cm}}{\pgfqpoint{0.173cm}{1.532cm}}{\pgfqpoint{0.137cm}{1.532cm}}
\pgfpathcurveto{\pgfqpoint{0.1cm}{1.532cm}}{\pgfqpoint{0.066cm}{1.518cm}}{\pgfqpoint{0.04cm}{1.492cm}}
\pgfpathcurveto{\pgfqpoint{0.014cm}{1.467cm}}{\pgfqpoint{0cm}{1.432cm}}{\pgfqpoint{0cm}{1.395cm}}
\pgfpathcurveto{\pgfqpoint{0cm}{1.359cm}}{\pgfqpoint{0.014cm}{1.324cm}}{\pgfqpoint{0.04cm}{1.299cm}}
\pgfpathcurveto{\pgfqpoint{0.066cm}{1.273cm}}{\pgfqpoint{0.1cm}{1.258cm}}{\pgfqpoint{0.137cm}{1.258cm}}
\pgfpathcurveto{\pgfqpoint{0.173cm}{1.258cm}}{\pgfqpoint{0.207cm}{1.273cm}}{\pgfqpoint{0.233cm}{1.299cm}}
\pgfpathcurveto{\pgfqpoint{0.259cm}{1.324cm}}{\pgfqpoint{0.273cm}{1.359cm}}{\pgfqpoint{0.273cm}{1.395cm}}
\pgfusepath{fill}
\begin{pgfscope}
\pgfsetdash{}{0cm}
\pgfsetlinewidth{0.818mm}
\pgfsetmiterlimit{7.0}
\pgfpathmoveto{\pgfqpoint{0.682cm}{0.671cm}}
\pgfpathlineto{\pgfqpoint{0.679cm}{1.418cm}}
\pgfusepath{stroke}
\end{pgfscope}
\pgfpathmoveto{\pgfqpoint{0.815cm}{1.399cm}}
\pgfpathcurveto{\pgfqpoint{0.815cm}{1.435cm}}{\pgfqpoint{0.801cm}{1.47cm}}{\pgfqpoint{0.775cm}{1.496cm}}
\pgfpathcurveto{\pgfqpoint{0.75cm}{1.521cm}}{\pgfqpoint{0.715cm}{1.536cm}}{\pgfqpoint{0.679cm}{1.536cm}}
\pgfpathcurveto{\pgfqpoint{0.643cm}{1.536cm}}{\pgfqpoint{0.608cm}{1.521cm}}{\pgfqpoint{0.582cm}{1.496cm}}
\pgfpathcurveto{\pgfqpoint{0.557cm}{1.47cm}}{\pgfqpoint{0.542cm}{1.435cm}}{\pgfqpoint{0.542cm}{1.399cm}}
\pgfpathcurveto{\pgfqpoint{0.542cm}{1.363cm}}{\pgfqpoint{0.557cm}{1.328cm}}{\pgfqpoint{0.582cm}{1.302cm}}
\pgfpathcurveto{\pgfqpoint{0.608cm}{1.276cm}}{\pgfqpoint{0.643cm}{1.262cm}}{\pgfqpoint{0.679cm}{1.262cm}}
\pgfpathcurveto{\pgfqpoint{0.715cm}{1.262cm}}{\pgfqpoint{0.75cm}{1.276cm}}{\pgfqpoint{0.775cm}{1.302cm}}
\pgfpathcurveto{\pgfqpoint{0.801cm}{1.328cm}}{\pgfqpoint{0.815cm}{1.363cm}}{\pgfqpoint{0.815cm}{1.399cm}}
\pgfusepath{fill}
\pgfpathmoveto{\pgfqpoint{1.345cm}{1.371cm}}
\pgfpathcurveto{\pgfqpoint{1.345cm}{1.408cm}}{\pgfqpoint{1.331cm}{1.442cm}}{\pgfqpoint{1.305cm}{1.468cm}}
\pgfpathcurveto{\pgfqpoint{1.28cm}{1.494cm}}{\pgfqpoint{1.245cm}{1.508cm}}{\pgfqpoint{1.209cm}{1.508cm}}
\pgfpathcurveto{\pgfqpoint{1.172cm}{1.508cm}}{\pgfqpoint{1.138cm}{1.494cm}}{\pgfqpoint{1.112cm}{1.468cm}}
\pgfpathcurveto{\pgfqpoint{1.087cm}{1.442cm}}{\pgfqpoint{1.072cm}{1.408cm}}{\pgfqpoint{1.072cm}{1.371cm}}
\pgfpathcurveto{\pgfqpoint{1.072cm}{1.335cm}}{\pgfqpoint{1.087cm}{1.3cm}}{\pgfqpoint{1.112cm}{1.274cm}}
\pgfpathcurveto{\pgfqpoint{1.138cm}{1.249cm}}{\pgfqpoint{1.172cm}{1.234cm}}{\pgfqpoint{1.209cm}{1.234cm}}
\pgfpathcurveto{\pgfqpoint{1.245cm}{1.234cm}}{\pgfqpoint{1.28cm}{1.249cm}}{\pgfqpoint{1.305cm}{1.274cm}}
\pgfpathcurveto{\pgfqpoint{1.331cm}{1.3cm}}{\pgfqpoint{1.345cm}{1.335cm}}{\pgfqpoint{1.345cm}{1.371cm}}
\pgfusepath{fill}
\begin{pgfscope}
\pgfsetdash{}{0cm}
\pgfsetlinewidth{0.818mm}
\pgfsetroundcap
\pgfsetmiterlimit{4.0}
\pgfpathmoveto{\pgfqpoint{0.682cm}{0.671cm}}
\pgfpathlineto{\pgfqpoint{0.682cm}{0.042cm}}
\pgfusepath{stroke}
\end{pgfscope}
\end{pgfscope}
\end{pgfscope}
\end{pgfscope}
\end{tikzpicture}}},X),
$$
and observe that all the other terms are well-defined.
Thus we obtain
\begin{align*}
3X(-X^{\!\resizebox{0.6em}{!}{
\begin{tikzpicture}
\pgfpathmoveto{\pgfqpoint{0cm}{-0.035cm}}
\pgfpathlineto{\pgfqpoint{1.376cm}{-0.035cm}}
\pgfpathlineto{\pgfqpoint{1.376cm}{1.552cm}}
\pgfpathlineto{\pgfqpoint{0cm}{1.552cm}}
\pgfpathclose
\pgfusepath{clip}
\begin{pgfscope}
\begin{pgfscope}
\pgfpathmoveto{\pgfqpoint{0cm}{-0.035cm}}
\pgfpathlineto{\pgfqpoint{1.376cm}{-0.035cm}}
\pgfpathlineto{\pgfqpoint{1.376cm}{1.552cm}}
\pgfpathlineto{\pgfqpoint{0cm}{1.552cm}}
\pgfpathclose
\pgfusepath{clip}
\begin{pgfscope}
\begin{pgfscope}
\pgfsetdash{}{0cm}
\pgfsetlinewidth{0.818mm}
\pgfsetroundcap
\pgfsetroundjoin
\pgfsetmiterlimit{7.0}
\definecolor{eps2pgf_color}{gray}{0}\pgfsetstrokecolor{eps2pgf_color}\pgfsetfillcolor{eps2pgf_color}
\pgfpathmoveto{\pgfqpoint{0.117cm}{1.421cm}}
\pgfpathlineto{\pgfqpoint{0.682cm}{0.671cm}}
\pgfpathlineto{\pgfqpoint{1.246cm}{1.421cm}}
\pgfusepath{stroke}
\end{pgfscope}
\definecolor{eps2pgf_color}{gray}{0}\pgfsetstrokecolor{eps2pgf_color}\pgfsetfillcolor{eps2pgf_color}
\pgfpathmoveto{\pgfqpoint{0.273cm}{1.395cm}}
\pgfpathcurveto{\pgfqpoint{0.273cm}{1.432cm}}{\pgfqpoint{0.259cm}{1.467cm}}{\pgfqpoint{0.233cm}{1.492cm}}
\pgfpathcurveto{\pgfqpoint{0.207cm}{1.518cm}}{\pgfqpoint{0.173cm}{1.532cm}}{\pgfqpoint{0.137cm}{1.532cm}}
\pgfpathcurveto{\pgfqpoint{0.1cm}{1.532cm}}{\pgfqpoint{0.066cm}{1.518cm}}{\pgfqpoint{0.04cm}{1.492cm}}
\pgfpathcurveto{\pgfqpoint{0.014cm}{1.467cm}}{\pgfqpoint{0cm}{1.432cm}}{\pgfqpoint{0cm}{1.395cm}}
\pgfpathcurveto{\pgfqpoint{0cm}{1.359cm}}{\pgfqpoint{0.014cm}{1.324cm}}{\pgfqpoint{0.04cm}{1.299cm}}
\pgfpathcurveto{\pgfqpoint{0.066cm}{1.273cm}}{\pgfqpoint{0.1cm}{1.258cm}}{\pgfqpoint{0.137cm}{1.258cm}}
\pgfpathcurveto{\pgfqpoint{0.173cm}{1.258cm}}{\pgfqpoint{0.207cm}{1.273cm}}{\pgfqpoint{0.233cm}{1.299cm}}
\pgfpathcurveto{\pgfqpoint{0.259cm}{1.324cm}}{\pgfqpoint{0.273cm}{1.359cm}}{\pgfqpoint{0.273cm}{1.395cm}}
\pgfusepath{fill}
\begin{pgfscope}
\pgfsetdash{}{0cm}
\pgfsetlinewidth{0.818mm}
\pgfsetmiterlimit{7.0}
\pgfpathmoveto{\pgfqpoint{0.682cm}{0.671cm}}
\pgfpathlineto{\pgfqpoint{0.679cm}{1.418cm}}
\pgfusepath{stroke}
\end{pgfscope}
\pgfpathmoveto{\pgfqpoint{0.815cm}{1.399cm}}
\pgfpathcurveto{\pgfqpoint{0.815cm}{1.435cm}}{\pgfqpoint{0.801cm}{1.47cm}}{\pgfqpoint{0.775cm}{1.496cm}}
\pgfpathcurveto{\pgfqpoint{0.75cm}{1.521cm}}{\pgfqpoint{0.715cm}{1.536cm}}{\pgfqpoint{0.679cm}{1.536cm}}
\pgfpathcurveto{\pgfqpoint{0.643cm}{1.536cm}}{\pgfqpoint{0.608cm}{1.521cm}}{\pgfqpoint{0.582cm}{1.496cm}}
\pgfpathcurveto{\pgfqpoint{0.557cm}{1.47cm}}{\pgfqpoint{0.542cm}{1.435cm}}{\pgfqpoint{0.542cm}{1.399cm}}
\pgfpathcurveto{\pgfqpoint{0.542cm}{1.363cm}}{\pgfqpoint{0.557cm}{1.328cm}}{\pgfqpoint{0.582cm}{1.302cm}}
\pgfpathcurveto{\pgfqpoint{0.608cm}{1.276cm}}{\pgfqpoint{0.643cm}{1.262cm}}{\pgfqpoint{0.679cm}{1.262cm}}
\pgfpathcurveto{\pgfqpoint{0.715cm}{1.262cm}}{\pgfqpoint{0.75cm}{1.276cm}}{\pgfqpoint{0.775cm}{1.302cm}}
\pgfpathcurveto{\pgfqpoint{0.801cm}{1.328cm}}{\pgfqpoint{0.815cm}{1.363cm}}{\pgfqpoint{0.815cm}{1.399cm}}
\pgfusepath{fill}
\pgfpathmoveto{\pgfqpoint{1.345cm}{1.371cm}}
\pgfpathcurveto{\pgfqpoint{1.345cm}{1.408cm}}{\pgfqpoint{1.331cm}{1.442cm}}{\pgfqpoint{1.305cm}{1.468cm}}
\pgfpathcurveto{\pgfqpoint{1.28cm}{1.494cm}}{\pgfqpoint{1.245cm}{1.508cm}}{\pgfqpoint{1.209cm}{1.508cm}}
\pgfpathcurveto{\pgfqpoint{1.172cm}{1.508cm}}{\pgfqpoint{1.138cm}{1.494cm}}{\pgfqpoint{1.112cm}{1.468cm}}
\pgfpathcurveto{\pgfqpoint{1.087cm}{1.442cm}}{\pgfqpoint{1.072cm}{1.408cm}}{\pgfqpoint{1.072cm}{1.371cm}}
\pgfpathcurveto{\pgfqpoint{1.072cm}{1.335cm}}{\pgfqpoint{1.087cm}{1.3cm}}{\pgfqpoint{1.112cm}{1.274cm}}
\pgfpathcurveto{\pgfqpoint{1.138cm}{1.249cm}}{\pgfqpoint{1.172cm}{1.234cm}}{\pgfqpoint{1.209cm}{1.234cm}}
\pgfpathcurveto{\pgfqpoint{1.245cm}{1.234cm}}{\pgfqpoint{1.28cm}{1.249cm}}{\pgfqpoint{1.305cm}{1.274cm}}
\pgfpathcurveto{\pgfqpoint{1.331cm}{1.3cm}}{\pgfqpoint{1.345cm}{1.335cm}}{\pgfqpoint{1.345cm}{1.371cm}}
\pgfusepath{fill}
\begin{pgfscope}
\pgfsetdash{}{0cm}
\pgfsetlinewidth{0.818mm}
\pgfsetroundcap
\pgfsetmiterlimit{4.0}
\pgfpathmoveto{\pgfqpoint{0.682cm}{0.671cm}}
\pgfpathlineto{\pgfqpoint{0.682cm}{0.042cm}}
\pgfusepath{stroke}
\end{pgfscope}
\end{pgfscope}
\end{pgfscope}
\end{pgfscope}
\end{tikzpicture}}}+\phi+\psi)^2&=
\rmm{3X\succ(X^{\!\resizebox{0.6em}{!}{
\begin{tikzpicture}
\pgfpathmoveto{\pgfqpoint{0cm}{-0.035cm}}
\pgfpathlineto{\pgfqpoint{1.376cm}{-0.035cm}}
\pgfpathlineto{\pgfqpoint{1.376cm}{1.552cm}}
\pgfpathlineto{\pgfqpoint{0cm}{1.552cm}}
\pgfpathclose
\pgfusepath{clip}
\begin{pgfscope}
\begin{pgfscope}
\pgfpathmoveto{\pgfqpoint{0cm}{-0.035cm}}
\pgfpathlineto{\pgfqpoint{1.376cm}{-0.035cm}}
\pgfpathlineto{\pgfqpoint{1.376cm}{1.552cm}}
\pgfpathlineto{\pgfqpoint{0cm}{1.552cm}}
\pgfpathclose
\pgfusepath{clip}
\begin{pgfscope}
\begin{pgfscope}
\pgfsetdash{}{0cm}
\pgfsetlinewidth{0.818mm}
\pgfsetroundcap
\pgfsetroundjoin
\pgfsetmiterlimit{7.0}
\definecolor{eps2pgf_color}{gray}{0}\pgfsetstrokecolor{eps2pgf_color}\pgfsetfillcolor{eps2pgf_color}
\pgfpathmoveto{\pgfqpoint{0.117cm}{1.421cm}}
\pgfpathlineto{\pgfqpoint{0.682cm}{0.671cm}}
\pgfpathlineto{\pgfqpoint{1.246cm}{1.421cm}}
\pgfusepath{stroke}
\end{pgfscope}
\definecolor{eps2pgf_color}{gray}{0}\pgfsetstrokecolor{eps2pgf_color}\pgfsetfillcolor{eps2pgf_color}
\pgfpathmoveto{\pgfqpoint{0.273cm}{1.395cm}}
\pgfpathcurveto{\pgfqpoint{0.273cm}{1.432cm}}{\pgfqpoint{0.259cm}{1.467cm}}{\pgfqpoint{0.233cm}{1.492cm}}
\pgfpathcurveto{\pgfqpoint{0.207cm}{1.518cm}}{\pgfqpoint{0.173cm}{1.532cm}}{\pgfqpoint{0.137cm}{1.532cm}}
\pgfpathcurveto{\pgfqpoint{0.1cm}{1.532cm}}{\pgfqpoint{0.066cm}{1.518cm}}{\pgfqpoint{0.04cm}{1.492cm}}
\pgfpathcurveto{\pgfqpoint{0.014cm}{1.467cm}}{\pgfqpoint{0cm}{1.432cm}}{\pgfqpoint{0cm}{1.395cm}}
\pgfpathcurveto{\pgfqpoint{0cm}{1.359cm}}{\pgfqpoint{0.014cm}{1.324cm}}{\pgfqpoint{0.04cm}{1.299cm}}
\pgfpathcurveto{\pgfqpoint{0.066cm}{1.273cm}}{\pgfqpoint{0.1cm}{1.258cm}}{\pgfqpoint{0.137cm}{1.258cm}}
\pgfpathcurveto{\pgfqpoint{0.173cm}{1.258cm}}{\pgfqpoint{0.207cm}{1.273cm}}{\pgfqpoint{0.233cm}{1.299cm}}
\pgfpathcurveto{\pgfqpoint{0.259cm}{1.324cm}}{\pgfqpoint{0.273cm}{1.359cm}}{\pgfqpoint{0.273cm}{1.395cm}}
\pgfusepath{fill}
\begin{pgfscope}
\pgfsetdash{}{0cm}
\pgfsetlinewidth{0.818mm}
\pgfsetmiterlimit{7.0}
\pgfpathmoveto{\pgfqpoint{0.682cm}{0.671cm}}
\pgfpathlineto{\pgfqpoint{0.679cm}{1.418cm}}
\pgfusepath{stroke}
\end{pgfscope}
\pgfpathmoveto{\pgfqpoint{0.815cm}{1.399cm}}
\pgfpathcurveto{\pgfqpoint{0.815cm}{1.435cm}}{\pgfqpoint{0.801cm}{1.47cm}}{\pgfqpoint{0.775cm}{1.496cm}}
\pgfpathcurveto{\pgfqpoint{0.75cm}{1.521cm}}{\pgfqpoint{0.715cm}{1.536cm}}{\pgfqpoint{0.679cm}{1.536cm}}
\pgfpathcurveto{\pgfqpoint{0.643cm}{1.536cm}}{\pgfqpoint{0.608cm}{1.521cm}}{\pgfqpoint{0.582cm}{1.496cm}}
\pgfpathcurveto{\pgfqpoint{0.557cm}{1.47cm}}{\pgfqpoint{0.542cm}{1.435cm}}{\pgfqpoint{0.542cm}{1.399cm}}
\pgfpathcurveto{\pgfqpoint{0.542cm}{1.363cm}}{\pgfqpoint{0.557cm}{1.328cm}}{\pgfqpoint{0.582cm}{1.302cm}}
\pgfpathcurveto{\pgfqpoint{0.608cm}{1.276cm}}{\pgfqpoint{0.643cm}{1.262cm}}{\pgfqpoint{0.679cm}{1.262cm}}
\pgfpathcurveto{\pgfqpoint{0.715cm}{1.262cm}}{\pgfqpoint{0.75cm}{1.276cm}}{\pgfqpoint{0.775cm}{1.302cm}}
\pgfpathcurveto{\pgfqpoint{0.801cm}{1.328cm}}{\pgfqpoint{0.815cm}{1.363cm}}{\pgfqpoint{0.815cm}{1.399cm}}
\pgfusepath{fill}
\pgfpathmoveto{\pgfqpoint{1.345cm}{1.371cm}}
\pgfpathcurveto{\pgfqpoint{1.345cm}{1.408cm}}{\pgfqpoint{1.331cm}{1.442cm}}{\pgfqpoint{1.305cm}{1.468cm}}
\pgfpathcurveto{\pgfqpoint{1.28cm}{1.494cm}}{\pgfqpoint{1.245cm}{1.508cm}}{\pgfqpoint{1.209cm}{1.508cm}}
\pgfpathcurveto{\pgfqpoint{1.172cm}{1.508cm}}{\pgfqpoint{1.138cm}{1.494cm}}{\pgfqpoint{1.112cm}{1.468cm}}
\pgfpathcurveto{\pgfqpoint{1.087cm}{1.442cm}}{\pgfqpoint{1.072cm}{1.408cm}}{\pgfqpoint{1.072cm}{1.371cm}}
\pgfpathcurveto{\pgfqpoint{1.072cm}{1.335cm}}{\pgfqpoint{1.087cm}{1.3cm}}{\pgfqpoint{1.112cm}{1.274cm}}
\pgfpathcurveto{\pgfqpoint{1.138cm}{1.249cm}}{\pgfqpoint{1.172cm}{1.234cm}}{\pgfqpoint{1.209cm}{1.234cm}}
\pgfpathcurveto{\pgfqpoint{1.245cm}{1.234cm}}{\pgfqpoint{1.28cm}{1.249cm}}{\pgfqpoint{1.305cm}{1.274cm}}
\pgfpathcurveto{\pgfqpoint{1.331cm}{1.3cm}}{\pgfqpoint{1.345cm}{1.335cm}}{\pgfqpoint{1.345cm}{1.371cm}}
\pgfusepath{fill}
\begin{pgfscope}
\pgfsetdash{}{0cm}
\pgfsetlinewidth{0.818mm}
\pgfsetroundcap
\pgfsetmiterlimit{4.0}
\pgfpathmoveto{\pgfqpoint{0.682cm}{0.671cm}}
\pgfpathlineto{\pgfqpoint{0.682cm}{0.042cm}}
\pgfusepath{stroke}
\end{pgfscope}
\end{pgfscope}
\end{pgfscope}
\end{pgfscope}
\end{tikzpicture}}})^2}+\rmm{3X\prec(X^{\!\resizebox{0.6em}{!}{
\begin{tikzpicture}
\pgfpathmoveto{\pgfqpoint{0cm}{-0.035cm}}
\pgfpathlineto{\pgfqpoint{1.376cm}{-0.035cm}}
\pgfpathlineto{\pgfqpoint{1.376cm}{1.552cm}}
\pgfpathlineto{\pgfqpoint{0cm}{1.552cm}}
\pgfpathclose
\pgfusepath{clip}
\begin{pgfscope}
\begin{pgfscope}
\pgfpathmoveto{\pgfqpoint{0cm}{-0.035cm}}
\pgfpathlineto{\pgfqpoint{1.376cm}{-0.035cm}}
\pgfpathlineto{\pgfqpoint{1.376cm}{1.552cm}}
\pgfpathlineto{\pgfqpoint{0cm}{1.552cm}}
\pgfpathclose
\pgfusepath{clip}
\begin{pgfscope}
\begin{pgfscope}
\pgfsetdash{}{0cm}
\pgfsetlinewidth{0.818mm}
\pgfsetroundcap
\pgfsetroundjoin
\pgfsetmiterlimit{7.0}
\definecolor{eps2pgf_color}{gray}{0}\pgfsetstrokecolor{eps2pgf_color}\pgfsetfillcolor{eps2pgf_color}
\pgfpathmoveto{\pgfqpoint{0.117cm}{1.421cm}}
\pgfpathlineto{\pgfqpoint{0.682cm}{0.671cm}}
\pgfpathlineto{\pgfqpoint{1.246cm}{1.421cm}}
\pgfusepath{stroke}
\end{pgfscope}
\definecolor{eps2pgf_color}{gray}{0}\pgfsetstrokecolor{eps2pgf_color}\pgfsetfillcolor{eps2pgf_color}
\pgfpathmoveto{\pgfqpoint{0.273cm}{1.395cm}}
\pgfpathcurveto{\pgfqpoint{0.273cm}{1.432cm}}{\pgfqpoint{0.259cm}{1.467cm}}{\pgfqpoint{0.233cm}{1.492cm}}
\pgfpathcurveto{\pgfqpoint{0.207cm}{1.518cm}}{\pgfqpoint{0.173cm}{1.532cm}}{\pgfqpoint{0.137cm}{1.532cm}}
\pgfpathcurveto{\pgfqpoint{0.1cm}{1.532cm}}{\pgfqpoint{0.066cm}{1.518cm}}{\pgfqpoint{0.04cm}{1.492cm}}
\pgfpathcurveto{\pgfqpoint{0.014cm}{1.467cm}}{\pgfqpoint{0cm}{1.432cm}}{\pgfqpoint{0cm}{1.395cm}}
\pgfpathcurveto{\pgfqpoint{0cm}{1.359cm}}{\pgfqpoint{0.014cm}{1.324cm}}{\pgfqpoint{0.04cm}{1.299cm}}
\pgfpathcurveto{\pgfqpoint{0.066cm}{1.273cm}}{\pgfqpoint{0.1cm}{1.258cm}}{\pgfqpoint{0.137cm}{1.258cm}}
\pgfpathcurveto{\pgfqpoint{0.173cm}{1.258cm}}{\pgfqpoint{0.207cm}{1.273cm}}{\pgfqpoint{0.233cm}{1.299cm}}
\pgfpathcurveto{\pgfqpoint{0.259cm}{1.324cm}}{\pgfqpoint{0.273cm}{1.359cm}}{\pgfqpoint{0.273cm}{1.395cm}}
\pgfusepath{fill}
\begin{pgfscope}
\pgfsetdash{}{0cm}
\pgfsetlinewidth{0.818mm}
\pgfsetmiterlimit{7.0}
\pgfpathmoveto{\pgfqpoint{0.682cm}{0.671cm}}
\pgfpathlineto{\pgfqpoint{0.679cm}{1.418cm}}
\pgfusepath{stroke}
\end{pgfscope}
\pgfpathmoveto{\pgfqpoint{0.815cm}{1.399cm}}
\pgfpathcurveto{\pgfqpoint{0.815cm}{1.435cm}}{\pgfqpoint{0.801cm}{1.47cm}}{\pgfqpoint{0.775cm}{1.496cm}}
\pgfpathcurveto{\pgfqpoint{0.75cm}{1.521cm}}{\pgfqpoint{0.715cm}{1.536cm}}{\pgfqpoint{0.679cm}{1.536cm}}
\pgfpathcurveto{\pgfqpoint{0.643cm}{1.536cm}}{\pgfqpoint{0.608cm}{1.521cm}}{\pgfqpoint{0.582cm}{1.496cm}}
\pgfpathcurveto{\pgfqpoint{0.557cm}{1.47cm}}{\pgfqpoint{0.542cm}{1.435cm}}{\pgfqpoint{0.542cm}{1.399cm}}
\pgfpathcurveto{\pgfqpoint{0.542cm}{1.363cm}}{\pgfqpoint{0.557cm}{1.328cm}}{\pgfqpoint{0.582cm}{1.302cm}}
\pgfpathcurveto{\pgfqpoint{0.608cm}{1.276cm}}{\pgfqpoint{0.643cm}{1.262cm}}{\pgfqpoint{0.679cm}{1.262cm}}
\pgfpathcurveto{\pgfqpoint{0.715cm}{1.262cm}}{\pgfqpoint{0.75cm}{1.276cm}}{\pgfqpoint{0.775cm}{1.302cm}}
\pgfpathcurveto{\pgfqpoint{0.801cm}{1.328cm}}{\pgfqpoint{0.815cm}{1.363cm}}{\pgfqpoint{0.815cm}{1.399cm}}
\pgfusepath{fill}
\pgfpathmoveto{\pgfqpoint{1.345cm}{1.371cm}}
\pgfpathcurveto{\pgfqpoint{1.345cm}{1.408cm}}{\pgfqpoint{1.331cm}{1.442cm}}{\pgfqpoint{1.305cm}{1.468cm}}
\pgfpathcurveto{\pgfqpoint{1.28cm}{1.494cm}}{\pgfqpoint{1.245cm}{1.508cm}}{\pgfqpoint{1.209cm}{1.508cm}}
\pgfpathcurveto{\pgfqpoint{1.172cm}{1.508cm}}{\pgfqpoint{1.138cm}{1.494cm}}{\pgfqpoint{1.112cm}{1.468cm}}
\pgfpathcurveto{\pgfqpoint{1.087cm}{1.442cm}}{\pgfqpoint{1.072cm}{1.408cm}}{\pgfqpoint{1.072cm}{1.371cm}}
\pgfpathcurveto{\pgfqpoint{1.072cm}{1.335cm}}{\pgfqpoint{1.087cm}{1.3cm}}{\pgfqpoint{1.112cm}{1.274cm}}
\pgfpathcurveto{\pgfqpoint{1.138cm}{1.249cm}}{\pgfqpoint{1.172cm}{1.234cm}}{\pgfqpoint{1.209cm}{1.234cm}}
\pgfpathcurveto{\pgfqpoint{1.245cm}{1.234cm}}{\pgfqpoint{1.28cm}{1.249cm}}{\pgfqpoint{1.305cm}{1.274cm}}
\pgfpathcurveto{\pgfqpoint{1.331cm}{1.3cm}}{\pgfqpoint{1.345cm}{1.335cm}}{\pgfqpoint{1.345cm}{1.371cm}}
\pgfusepath{fill}
\begin{pgfscope}
\pgfsetdash{}{0cm}
\pgfsetlinewidth{0.818mm}
\pgfsetroundcap
\pgfsetmiterlimit{4.0}
\pgfpathmoveto{\pgfqpoint{0.682cm}{0.671cm}}
\pgfpathlineto{\pgfqpoint{0.682cm}{0.042cm}}
\pgfusepath{stroke}
\end{pgfscope}
\end{pgfscope}
\end{pgfscope}
\end{pgfscope}
\end{tikzpicture}}})^2}+\rmm{6X^{\!\resizebox{0.6em}{!}{
\begin{tikzpicture}
\pgfpathmoveto{\pgfqpoint{0cm}{-0.035cm}}
\pgfpathlineto{\pgfqpoint{1.376cm}{-0.035cm}}
\pgfpathlineto{\pgfqpoint{1.376cm}{1.552cm}}
\pgfpathlineto{\pgfqpoint{0cm}{1.552cm}}
\pgfpathclose
\pgfusepath{clip}
\begin{pgfscope}
\begin{pgfscope}
\pgfpathmoveto{\pgfqpoint{0cm}{-0.035cm}}
\pgfpathlineto{\pgfqpoint{1.376cm}{-0.035cm}}
\pgfpathlineto{\pgfqpoint{1.376cm}{1.552cm}}
\pgfpathlineto{\pgfqpoint{0cm}{1.552cm}}
\pgfpathclose
\pgfusepath{clip}
\begin{pgfscope}
\begin{pgfscope}
\pgfsetdash{}{0cm}
\pgfsetlinewidth{0.818mm}
\pgfsetroundcap
\pgfsetroundjoin
\pgfsetmiterlimit{7.0}
\definecolor{eps2pgf_color}{gray}{0}\pgfsetstrokecolor{eps2pgf_color}\pgfsetfillcolor{eps2pgf_color}
\pgfpathmoveto{\pgfqpoint{0.117cm}{1.421cm}}
\pgfpathlineto{\pgfqpoint{0.682cm}{0.671cm}}
\pgfpathlineto{\pgfqpoint{1.246cm}{1.421cm}}
\pgfusepath{stroke}
\end{pgfscope}
\definecolor{eps2pgf_color}{gray}{0}\pgfsetstrokecolor{eps2pgf_color}\pgfsetfillcolor{eps2pgf_color}
\pgfpathmoveto{\pgfqpoint{0.273cm}{1.395cm}}
\pgfpathcurveto{\pgfqpoint{0.273cm}{1.432cm}}{\pgfqpoint{0.259cm}{1.467cm}}{\pgfqpoint{0.233cm}{1.492cm}}
\pgfpathcurveto{\pgfqpoint{0.207cm}{1.518cm}}{\pgfqpoint{0.173cm}{1.532cm}}{\pgfqpoint{0.137cm}{1.532cm}}
\pgfpathcurveto{\pgfqpoint{0.1cm}{1.532cm}}{\pgfqpoint{0.066cm}{1.518cm}}{\pgfqpoint{0.04cm}{1.492cm}}
\pgfpathcurveto{\pgfqpoint{0.014cm}{1.467cm}}{\pgfqpoint{0cm}{1.432cm}}{\pgfqpoint{0cm}{1.395cm}}
\pgfpathcurveto{\pgfqpoint{0cm}{1.359cm}}{\pgfqpoint{0.014cm}{1.324cm}}{\pgfqpoint{0.04cm}{1.299cm}}
\pgfpathcurveto{\pgfqpoint{0.066cm}{1.273cm}}{\pgfqpoint{0.1cm}{1.258cm}}{\pgfqpoint{0.137cm}{1.258cm}}
\pgfpathcurveto{\pgfqpoint{0.173cm}{1.258cm}}{\pgfqpoint{0.207cm}{1.273cm}}{\pgfqpoint{0.233cm}{1.299cm}}
\pgfpathcurveto{\pgfqpoint{0.259cm}{1.324cm}}{\pgfqpoint{0.273cm}{1.359cm}}{\pgfqpoint{0.273cm}{1.395cm}}
\pgfusepath{fill}
\begin{pgfscope}
\pgfsetdash{}{0cm}
\pgfsetlinewidth{0.818mm}
\pgfsetmiterlimit{7.0}
\pgfpathmoveto{\pgfqpoint{0.682cm}{0.671cm}}
\pgfpathlineto{\pgfqpoint{0.679cm}{1.418cm}}
\pgfusepath{stroke}
\end{pgfscope}
\pgfpathmoveto{\pgfqpoint{0.815cm}{1.399cm}}
\pgfpathcurveto{\pgfqpoint{0.815cm}{1.435cm}}{\pgfqpoint{0.801cm}{1.47cm}}{\pgfqpoint{0.775cm}{1.496cm}}
\pgfpathcurveto{\pgfqpoint{0.75cm}{1.521cm}}{\pgfqpoint{0.715cm}{1.536cm}}{\pgfqpoint{0.679cm}{1.536cm}}
\pgfpathcurveto{\pgfqpoint{0.643cm}{1.536cm}}{\pgfqpoint{0.608cm}{1.521cm}}{\pgfqpoint{0.582cm}{1.496cm}}
\pgfpathcurveto{\pgfqpoint{0.557cm}{1.47cm}}{\pgfqpoint{0.542cm}{1.435cm}}{\pgfqpoint{0.542cm}{1.399cm}}
\pgfpathcurveto{\pgfqpoint{0.542cm}{1.363cm}}{\pgfqpoint{0.557cm}{1.328cm}}{\pgfqpoint{0.582cm}{1.302cm}}
\pgfpathcurveto{\pgfqpoint{0.608cm}{1.276cm}}{\pgfqpoint{0.643cm}{1.262cm}}{\pgfqpoint{0.679cm}{1.262cm}}
\pgfpathcurveto{\pgfqpoint{0.715cm}{1.262cm}}{\pgfqpoint{0.75cm}{1.276cm}}{\pgfqpoint{0.775cm}{1.302cm}}
\pgfpathcurveto{\pgfqpoint{0.801cm}{1.328cm}}{\pgfqpoint{0.815cm}{1.363cm}}{\pgfqpoint{0.815cm}{1.399cm}}
\pgfusepath{fill}
\pgfpathmoveto{\pgfqpoint{1.345cm}{1.371cm}}
\pgfpathcurveto{\pgfqpoint{1.345cm}{1.408cm}}{\pgfqpoint{1.331cm}{1.442cm}}{\pgfqpoint{1.305cm}{1.468cm}}
\pgfpathcurveto{\pgfqpoint{1.28cm}{1.494cm}}{\pgfqpoint{1.245cm}{1.508cm}}{\pgfqpoint{1.209cm}{1.508cm}}
\pgfpathcurveto{\pgfqpoint{1.172cm}{1.508cm}}{\pgfqpoint{1.138cm}{1.494cm}}{\pgfqpoint{1.112cm}{1.468cm}}
\pgfpathcurveto{\pgfqpoint{1.087cm}{1.442cm}}{\pgfqpoint{1.072cm}{1.408cm}}{\pgfqpoint{1.072cm}{1.371cm}}
\pgfpathcurveto{\pgfqpoint{1.072cm}{1.335cm}}{\pgfqpoint{1.087cm}{1.3cm}}{\pgfqpoint{1.112cm}{1.274cm}}
\pgfpathcurveto{\pgfqpoint{1.138cm}{1.249cm}}{\pgfqpoint{1.172cm}{1.234cm}}{\pgfqpoint{1.209cm}{1.234cm}}
\pgfpathcurveto{\pgfqpoint{1.245cm}{1.234cm}}{\pgfqpoint{1.28cm}{1.249cm}}{\pgfqpoint{1.305cm}{1.274cm}}
\pgfpathcurveto{\pgfqpoint{1.331cm}{1.3cm}}{\pgfqpoint{1.345cm}{1.335cm}}{\pgfqpoint{1.345cm}{1.371cm}}
\pgfusepath{fill}
\begin{pgfscope}
\pgfsetdash{}{0cm}
\pgfsetlinewidth{0.818mm}
\pgfsetroundcap
\pgfsetmiterlimit{4.0}
\pgfpathmoveto{\pgfqpoint{0.682cm}{0.671cm}}
\pgfpathlineto{\pgfqpoint{0.682cm}{0.042cm}}
\pgfusepath{stroke}
\end{pgfscope}
\end{pgfscope}
\end{pgfscope}
\end{pgfscope}
\end{tikzpicture}}}X^{\!\resizebox{!}{.8em}{
\begin{tikzpicture}
\pgfpathmoveto{\pgfqpoint{0cm}{-0.035cm}}
\pgfpathlineto{\pgfqpoint{1.976cm}{-0.035cm}}
\pgfpathlineto{\pgfqpoint{1.976cm}{1.94cm}}
\pgfpathlineto{\pgfqpoint{0cm}{1.94cm}}
\pgfpathclose
\pgfusepath{clip}
\begin{pgfscope}
\begin{pgfscope}
\pgfpathmoveto{\pgfqpoint{0cm}{-0.035cm}}
\pgfpathlineto{\pgfqpoint{1.976cm}{-0.035cm}}
\pgfpathlineto{\pgfqpoint{1.976cm}{1.94cm}}
\pgfpathlineto{\pgfqpoint{0cm}{1.94cm}}
\pgfpathclose
\pgfusepath{clip}
\begin{pgfscope}
\begin{pgfscope}
\pgfsetdash{}{0cm}
\pgfsetlinewidth{0.818mm}
\pgfsetroundcap
\pgfsetroundjoin
\pgfsetmiterlimit{7.0}
\definecolor{eps2pgf_color}{gray}{0}\pgfsetstrokecolor{eps2pgf_color}\pgfsetfillcolor{eps2pgf_color}
\pgfpathmoveto{\pgfqpoint{0.117cm}{1.815cm}}
\pgfpathlineto{\pgfqpoint{0.682cm}{1.065cm}}
\pgfpathlineto{\pgfqpoint{1.246cm}{1.815cm}}
\pgfusepath{stroke}
\end{pgfscope}
\definecolor{eps2pgf_color}{gray}{0}\pgfsetstrokecolor{eps2pgf_color}\pgfsetfillcolor{eps2pgf_color}
\pgfpathmoveto{\pgfqpoint{0.273cm}{1.789cm}}
\pgfpathcurveto{\pgfqpoint{0.273cm}{1.825cm}}{\pgfqpoint{0.259cm}{1.86cm}}{\pgfqpoint{0.233cm}{1.886cm}}
\pgfpathcurveto{\pgfqpoint{0.207cm}{1.912cm}}{\pgfqpoint{0.173cm}{1.926cm}}{\pgfqpoint{0.137cm}{1.926cm}}
\pgfpathcurveto{\pgfqpoint{0.1cm}{1.926cm}}{\pgfqpoint{0.066cm}{1.912cm}}{\pgfqpoint{0.04cm}{1.886cm}}
\pgfpathcurveto{\pgfqpoint{0.014cm}{1.86cm}}{\pgfqpoint{0cm}{1.825cm}}{\pgfqpoint{0cm}{1.789cm}}
\pgfpathcurveto{\pgfqpoint{0cm}{1.753cm}}{\pgfqpoint{0.014cm}{1.718cm}}{\pgfqpoint{0.04cm}{1.692cm}}
\pgfpathcurveto{\pgfqpoint{0.066cm}{1.667cm}}{\pgfqpoint{0.1cm}{1.652cm}}{\pgfqpoint{0.137cm}{1.652cm}}
\pgfpathcurveto{\pgfqpoint{0.173cm}{1.652cm}}{\pgfqpoint{0.207cm}{1.667cm}}{\pgfqpoint{0.233cm}{1.692cm}}
\pgfpathcurveto{\pgfqpoint{0.259cm}{1.718cm}}{\pgfqpoint{0.273cm}{1.753cm}}{\pgfqpoint{0.273cm}{1.789cm}}
\pgfusepath{fill}
\begin{pgfscope}
\pgfsetdash{}{0cm}
\pgfsetlinewidth{0.818mm}
\pgfsetmiterlimit{7.0}
\pgfpathmoveto{\pgfqpoint{0.682cm}{1.065cm}}
\pgfpathlineto{\pgfqpoint{0.679cm}{1.812cm}}
\pgfusepath{stroke}
\end{pgfscope}
\pgfpathmoveto{\pgfqpoint{0.815cm}{1.793cm}}
\pgfpathcurveto{\pgfqpoint{0.815cm}{1.829cm}}{\pgfqpoint{0.801cm}{1.864cm}}{\pgfqpoint{0.775cm}{1.89cm}}
\pgfpathcurveto{\pgfqpoint{0.75cm}{1.915cm}}{\pgfqpoint{0.715cm}{1.93cm}}{\pgfqpoint{0.679cm}{1.93cm}}
\pgfpathcurveto{\pgfqpoint{0.643cm}{1.93cm}}{\pgfqpoint{0.608cm}{1.915cm}}{\pgfqpoint{0.582cm}{1.89cm}}
\pgfpathcurveto{\pgfqpoint{0.557cm}{1.864cm}}{\pgfqpoint{0.542cm}{1.829cm}}{\pgfqpoint{0.542cm}{1.793cm}}
\pgfpathcurveto{\pgfqpoint{0.542cm}{1.756cm}}{\pgfqpoint{0.557cm}{1.722cm}}{\pgfqpoint{0.582cm}{1.696cm}}
\pgfpathcurveto{\pgfqpoint{0.608cm}{1.67cm}}{\pgfqpoint{0.643cm}{1.656cm}}{\pgfqpoint{0.679cm}{1.656cm}}
\pgfpathcurveto{\pgfqpoint{0.715cm}{1.656cm}}{\pgfqpoint{0.75cm}{1.67cm}}{\pgfqpoint{0.775cm}{1.696cm}}
\pgfpathcurveto{\pgfqpoint{0.801cm}{1.722cm}}{\pgfqpoint{0.815cm}{1.756cm}}{\pgfqpoint{0.815cm}{1.793cm}}
\pgfusepath{fill}
\pgfpathmoveto{\pgfqpoint{1.345cm}{1.765cm}}
\pgfpathcurveto{\pgfqpoint{1.345cm}{1.801cm}}{\pgfqpoint{1.331cm}{1.836cm}}{\pgfqpoint{1.305cm}{1.862cm}}
\pgfpathcurveto{\pgfqpoint{1.28cm}{1.887cm}}{\pgfqpoint{1.245cm}{1.902cm}}{\pgfqpoint{1.209cm}{1.902cm}}
\pgfpathcurveto{\pgfqpoint{1.172cm}{1.902cm}}{\pgfqpoint{1.138cm}{1.887cm}}{\pgfqpoint{1.112cm}{1.862cm}}
\pgfpathcurveto{\pgfqpoint{1.087cm}{1.836cm}}{\pgfqpoint{1.072cm}{1.801cm}}{\pgfqpoint{1.072cm}{1.765cm}}
\pgfpathcurveto{\pgfqpoint{1.072cm}{1.728cm}}{\pgfqpoint{1.087cm}{1.694cm}}{\pgfqpoint{1.112cm}{1.668cm}}
\pgfpathcurveto{\pgfqpoint{1.138cm}{1.642cm}}{\pgfqpoint{1.172cm}{1.628cm}}{\pgfqpoint{1.209cm}{1.628cm}}
\pgfpathcurveto{\pgfqpoint{1.245cm}{1.628cm}}{\pgfqpoint{1.28cm}{1.642cm}}{\pgfqpoint{1.305cm}{1.668cm}}
\pgfpathcurveto{\pgfqpoint{1.331cm}{1.694cm}}{\pgfqpoint{1.345cm}{1.728cm}}{\pgfqpoint{1.345cm}{1.765cm}}
\pgfusepath{fill}
\begin{pgfscope}
\pgfsetdash{}{0cm}
\pgfsetlinewidth{0.818mm}
\pgfsetroundcap
\pgfsetroundjoin
\pgfsetmiterlimit{7.0}
\pgfpathmoveto{\pgfqpoint{0.682cm}{1.065cm}}
\pgfpathlineto{\pgfqpoint{1.246cm}{0.315cm}}
\pgfpathlineto{\pgfqpoint{1.811cm}{1.065cm}}
\pgfusepath{stroke}
\end{pgfscope}
\pgfpathmoveto{\pgfqpoint{1.948cm}{1.065cm}}
\pgfpathcurveto{\pgfqpoint{1.948cm}{1.101cm}}{\pgfqpoint{1.933cm}{1.136cm}}{\pgfqpoint{1.907cm}{1.162cm}}
\pgfpathcurveto{\pgfqpoint{1.882cm}{1.187cm}}{\pgfqpoint{1.847cm}{1.202cm}}{\pgfqpoint{1.811cm}{1.202cm}}
\pgfpathcurveto{\pgfqpoint{1.775cm}{1.202cm}}{\pgfqpoint{1.74cm}{1.187cm}}{\pgfqpoint{1.714cm}{1.162cm}}
\pgfpathcurveto{\pgfqpoint{1.689cm}{1.136cm}}{\pgfqpoint{1.674cm}{1.101cm}}{\pgfqpoint{1.674cm}{1.065cm}}
\pgfpathcurveto{\pgfqpoint{1.674cm}{1.029cm}}{\pgfqpoint{1.689cm}{0.994cm}}{\pgfqpoint{1.714cm}{0.968cm}}
\pgfpathcurveto{\pgfqpoint{1.74cm}{0.942cm}}{\pgfqpoint{1.775cm}{0.928cm}}{\pgfqpoint{1.811cm}{0.928cm}}
\pgfpathcurveto{\pgfqpoint{1.847cm}{0.928cm}}{\pgfqpoint{1.882cm}{0.942cm}}{\pgfqpoint{1.907cm}{0.968cm}}
\pgfpathcurveto{\pgfqpoint{1.933cm}{0.994cm}}{\pgfqpoint{1.948cm}{1.029cm}}{\pgfqpoint{1.948cm}{1.065cm}}
\pgfusepath{fill}
\begin{pgfscope}
\pgfsetdash{}{0cm}
\pgfsetlinewidth{0.818mm}
\pgfsetmiterlimit{4.0}
\pgfpathmoveto{\pgfqpoint{1.383cm}{0.178cm}}
\pgfpathcurveto{\pgfqpoint{1.383cm}{0.214cm}}{\pgfqpoint{1.369cm}{0.249cm}}{\pgfqpoint{1.343cm}{0.275cm}}
\pgfpathcurveto{\pgfqpoint{1.317cm}{0.3cm}}{\pgfqpoint{1.283cm}{0.315cm}}{\pgfqpoint{1.246cm}{0.315cm}}
\pgfpathcurveto{\pgfqpoint{1.21cm}{0.315cm}}{\pgfqpoint{1.175cm}{0.3cm}}{\pgfqpoint{1.15cm}{0.275cm}}
\pgfpathcurveto{\pgfqpoint{1.124cm}{0.249cm}}{\pgfqpoint{1.11cm}{0.214cm}}{\pgfqpoint{1.11cm}{0.178cm}}
\pgfpathcurveto{\pgfqpoint{1.11cm}{0.141cm}}{\pgfqpoint{1.124cm}{0.107cm}}{\pgfqpoint{1.15cm}{0.081cm}}
\pgfpathcurveto{\pgfqpoint{1.175cm}{0.055cm}}{\pgfqpoint{1.21cm}{0.041cm}}{\pgfqpoint{1.246cm}{0.041cm}}
\pgfpathcurveto{\pgfqpoint{1.283cm}{0.041cm}}{\pgfqpoint{1.317cm}{0.055cm}}{\pgfqpoint{1.343cm}{0.081cm}}
\pgfpathcurveto{\pgfqpoint{1.369cm}{0.107cm}}{\pgfqpoint{1.383cm}{0.141cm}}{\pgfqpoint{1.383cm}{0.178cm}}
\pgfusepath{stroke}
\end{pgfscope}
\end{pgfscope}
\end{pgfscope}
\end{pgfscope}
\end{tikzpicture}}}}\\
&\quad+\rmb{6\mathrm{com}(X^{\!\resizebox{0.6em}{!}{
\begin{tikzpicture}
\pgfpathmoveto{\pgfqpoint{0cm}{-0.035cm}}
\pgfpathlineto{\pgfqpoint{1.376cm}{-0.035cm}}
\pgfpathlineto{\pgfqpoint{1.376cm}{1.552cm}}
\pgfpathlineto{\pgfqpoint{0cm}{1.552cm}}
\pgfpathclose
\pgfusepath{clip}
\begin{pgfscope}
\begin{pgfscope}
\pgfpathmoveto{\pgfqpoint{0cm}{-0.035cm}}
\pgfpathlineto{\pgfqpoint{1.376cm}{-0.035cm}}
\pgfpathlineto{\pgfqpoint{1.376cm}{1.552cm}}
\pgfpathlineto{\pgfqpoint{0cm}{1.552cm}}
\pgfpathclose
\pgfusepath{clip}
\begin{pgfscope}
\begin{pgfscope}
\pgfsetdash{}{0cm}
\pgfsetlinewidth{0.818mm}
\pgfsetroundcap
\pgfsetroundjoin
\pgfsetmiterlimit{7.0}
\definecolor{eps2pgf_color}{gray}{0}\pgfsetstrokecolor{eps2pgf_color}\pgfsetfillcolor{eps2pgf_color}
\pgfpathmoveto{\pgfqpoint{0.117cm}{1.421cm}}
\pgfpathlineto{\pgfqpoint{0.682cm}{0.671cm}}
\pgfpathlineto{\pgfqpoint{1.246cm}{1.421cm}}
\pgfusepath{stroke}
\end{pgfscope}
\definecolor{eps2pgf_color}{gray}{0}\pgfsetstrokecolor{eps2pgf_color}\pgfsetfillcolor{eps2pgf_color}
\pgfpathmoveto{\pgfqpoint{0.273cm}{1.395cm}}
\pgfpathcurveto{\pgfqpoint{0.273cm}{1.432cm}}{\pgfqpoint{0.259cm}{1.467cm}}{\pgfqpoint{0.233cm}{1.492cm}}
\pgfpathcurveto{\pgfqpoint{0.207cm}{1.518cm}}{\pgfqpoint{0.173cm}{1.532cm}}{\pgfqpoint{0.137cm}{1.532cm}}
\pgfpathcurveto{\pgfqpoint{0.1cm}{1.532cm}}{\pgfqpoint{0.066cm}{1.518cm}}{\pgfqpoint{0.04cm}{1.492cm}}
\pgfpathcurveto{\pgfqpoint{0.014cm}{1.467cm}}{\pgfqpoint{0cm}{1.432cm}}{\pgfqpoint{0cm}{1.395cm}}
\pgfpathcurveto{\pgfqpoint{0cm}{1.359cm}}{\pgfqpoint{0.014cm}{1.324cm}}{\pgfqpoint{0.04cm}{1.299cm}}
\pgfpathcurveto{\pgfqpoint{0.066cm}{1.273cm}}{\pgfqpoint{0.1cm}{1.258cm}}{\pgfqpoint{0.137cm}{1.258cm}}
\pgfpathcurveto{\pgfqpoint{0.173cm}{1.258cm}}{\pgfqpoint{0.207cm}{1.273cm}}{\pgfqpoint{0.233cm}{1.299cm}}
\pgfpathcurveto{\pgfqpoint{0.259cm}{1.324cm}}{\pgfqpoint{0.273cm}{1.359cm}}{\pgfqpoint{0.273cm}{1.395cm}}
\pgfusepath{fill}
\begin{pgfscope}
\pgfsetdash{}{0cm}
\pgfsetlinewidth{0.818mm}
\pgfsetmiterlimit{7.0}
\pgfpathmoveto{\pgfqpoint{0.682cm}{0.671cm}}
\pgfpathlineto{\pgfqpoint{0.679cm}{1.418cm}}
\pgfusepath{stroke}
\end{pgfscope}
\pgfpathmoveto{\pgfqpoint{0.815cm}{1.399cm}}
\pgfpathcurveto{\pgfqpoint{0.815cm}{1.435cm}}{\pgfqpoint{0.801cm}{1.47cm}}{\pgfqpoint{0.775cm}{1.496cm}}
\pgfpathcurveto{\pgfqpoint{0.75cm}{1.521cm}}{\pgfqpoint{0.715cm}{1.536cm}}{\pgfqpoint{0.679cm}{1.536cm}}
\pgfpathcurveto{\pgfqpoint{0.643cm}{1.536cm}}{\pgfqpoint{0.608cm}{1.521cm}}{\pgfqpoint{0.582cm}{1.496cm}}
\pgfpathcurveto{\pgfqpoint{0.557cm}{1.47cm}}{\pgfqpoint{0.542cm}{1.435cm}}{\pgfqpoint{0.542cm}{1.399cm}}
\pgfpathcurveto{\pgfqpoint{0.542cm}{1.363cm}}{\pgfqpoint{0.557cm}{1.328cm}}{\pgfqpoint{0.582cm}{1.302cm}}
\pgfpathcurveto{\pgfqpoint{0.608cm}{1.276cm}}{\pgfqpoint{0.643cm}{1.262cm}}{\pgfqpoint{0.679cm}{1.262cm}}
\pgfpathcurveto{\pgfqpoint{0.715cm}{1.262cm}}{\pgfqpoint{0.75cm}{1.276cm}}{\pgfqpoint{0.775cm}{1.302cm}}
\pgfpathcurveto{\pgfqpoint{0.801cm}{1.328cm}}{\pgfqpoint{0.815cm}{1.363cm}}{\pgfqpoint{0.815cm}{1.399cm}}
\pgfusepath{fill}
\pgfpathmoveto{\pgfqpoint{1.345cm}{1.371cm}}
\pgfpathcurveto{\pgfqpoint{1.345cm}{1.408cm}}{\pgfqpoint{1.331cm}{1.442cm}}{\pgfqpoint{1.305cm}{1.468cm}}
\pgfpathcurveto{\pgfqpoint{1.28cm}{1.494cm}}{\pgfqpoint{1.245cm}{1.508cm}}{\pgfqpoint{1.209cm}{1.508cm}}
\pgfpathcurveto{\pgfqpoint{1.172cm}{1.508cm}}{\pgfqpoint{1.138cm}{1.494cm}}{\pgfqpoint{1.112cm}{1.468cm}}
\pgfpathcurveto{\pgfqpoint{1.087cm}{1.442cm}}{\pgfqpoint{1.072cm}{1.408cm}}{\pgfqpoint{1.072cm}{1.371cm}}
\pgfpathcurveto{\pgfqpoint{1.072cm}{1.335cm}}{\pgfqpoint{1.087cm}{1.3cm}}{\pgfqpoint{1.112cm}{1.274cm}}
\pgfpathcurveto{\pgfqpoint{1.138cm}{1.249cm}}{\pgfqpoint{1.172cm}{1.234cm}}{\pgfqpoint{1.209cm}{1.234cm}}
\pgfpathcurveto{\pgfqpoint{1.245cm}{1.234cm}}{\pgfqpoint{1.28cm}{1.249cm}}{\pgfqpoint{1.305cm}{1.274cm}}
\pgfpathcurveto{\pgfqpoint{1.331cm}{1.3cm}}{\pgfqpoint{1.345cm}{1.335cm}}{\pgfqpoint{1.345cm}{1.371cm}}
\pgfusepath{fill}
\begin{pgfscope}
\pgfsetdash{}{0cm}
\pgfsetlinewidth{0.818mm}
\pgfsetroundcap
\pgfsetmiterlimit{4.0}
\pgfpathmoveto{\pgfqpoint{0.682cm}{0.671cm}}
\pgfpathlineto{\pgfqpoint{0.682cm}{0.042cm}}
\pgfusepath{stroke}
\end{pgfscope}
\end{pgfscope}
\end{pgfscope}
\end{pgfscope}
\end{tikzpicture}}},X^{\!\resizebox{0.6em}{!}{
\begin{tikzpicture}
\pgfpathmoveto{\pgfqpoint{0cm}{-0.035cm}}
\pgfpathlineto{\pgfqpoint{1.376cm}{-0.035cm}}
\pgfpathlineto{\pgfqpoint{1.376cm}{1.552cm}}
\pgfpathlineto{\pgfqpoint{0cm}{1.552cm}}
\pgfpathclose
\pgfusepath{clip}
\begin{pgfscope}
\begin{pgfscope}
\pgfpathmoveto{\pgfqpoint{0cm}{-0.035cm}}
\pgfpathlineto{\pgfqpoint{1.376cm}{-0.035cm}}
\pgfpathlineto{\pgfqpoint{1.376cm}{1.552cm}}
\pgfpathlineto{\pgfqpoint{0cm}{1.552cm}}
\pgfpathclose
\pgfusepath{clip}
\begin{pgfscope}
\begin{pgfscope}
\pgfsetdash{}{0cm}
\pgfsetlinewidth{0.818mm}
\pgfsetroundcap
\pgfsetroundjoin
\pgfsetmiterlimit{7.0}
\definecolor{eps2pgf_color}{gray}{0}\pgfsetstrokecolor{eps2pgf_color}\pgfsetfillcolor{eps2pgf_color}
\pgfpathmoveto{\pgfqpoint{0.117cm}{1.421cm}}
\pgfpathlineto{\pgfqpoint{0.682cm}{0.671cm}}
\pgfpathlineto{\pgfqpoint{1.246cm}{1.421cm}}
\pgfusepath{stroke}
\end{pgfscope}
\definecolor{eps2pgf_color}{gray}{0}\pgfsetstrokecolor{eps2pgf_color}\pgfsetfillcolor{eps2pgf_color}
\pgfpathmoveto{\pgfqpoint{0.273cm}{1.395cm}}
\pgfpathcurveto{\pgfqpoint{0.273cm}{1.432cm}}{\pgfqpoint{0.259cm}{1.467cm}}{\pgfqpoint{0.233cm}{1.492cm}}
\pgfpathcurveto{\pgfqpoint{0.207cm}{1.518cm}}{\pgfqpoint{0.173cm}{1.532cm}}{\pgfqpoint{0.137cm}{1.532cm}}
\pgfpathcurveto{\pgfqpoint{0.1cm}{1.532cm}}{\pgfqpoint{0.066cm}{1.518cm}}{\pgfqpoint{0.04cm}{1.492cm}}
\pgfpathcurveto{\pgfqpoint{0.014cm}{1.467cm}}{\pgfqpoint{0cm}{1.432cm}}{\pgfqpoint{0cm}{1.395cm}}
\pgfpathcurveto{\pgfqpoint{0cm}{1.359cm}}{\pgfqpoint{0.014cm}{1.324cm}}{\pgfqpoint{0.04cm}{1.299cm}}
\pgfpathcurveto{\pgfqpoint{0.066cm}{1.273cm}}{\pgfqpoint{0.1cm}{1.258cm}}{\pgfqpoint{0.137cm}{1.258cm}}
\pgfpathcurveto{\pgfqpoint{0.173cm}{1.258cm}}{\pgfqpoint{0.207cm}{1.273cm}}{\pgfqpoint{0.233cm}{1.299cm}}
\pgfpathcurveto{\pgfqpoint{0.259cm}{1.324cm}}{\pgfqpoint{0.273cm}{1.359cm}}{\pgfqpoint{0.273cm}{1.395cm}}
\pgfusepath{fill}
\begin{pgfscope}
\pgfsetdash{}{0cm}
\pgfsetlinewidth{0.818mm}
\pgfsetmiterlimit{7.0}
\pgfpathmoveto{\pgfqpoint{0.682cm}{0.671cm}}
\pgfpathlineto{\pgfqpoint{0.679cm}{1.418cm}}
\pgfusepath{stroke}
\end{pgfscope}
\pgfpathmoveto{\pgfqpoint{0.815cm}{1.399cm}}
\pgfpathcurveto{\pgfqpoint{0.815cm}{1.435cm}}{\pgfqpoint{0.801cm}{1.47cm}}{\pgfqpoint{0.775cm}{1.496cm}}
\pgfpathcurveto{\pgfqpoint{0.75cm}{1.521cm}}{\pgfqpoint{0.715cm}{1.536cm}}{\pgfqpoint{0.679cm}{1.536cm}}
\pgfpathcurveto{\pgfqpoint{0.643cm}{1.536cm}}{\pgfqpoint{0.608cm}{1.521cm}}{\pgfqpoint{0.582cm}{1.496cm}}
\pgfpathcurveto{\pgfqpoint{0.557cm}{1.47cm}}{\pgfqpoint{0.542cm}{1.435cm}}{\pgfqpoint{0.542cm}{1.399cm}}
\pgfpathcurveto{\pgfqpoint{0.542cm}{1.363cm}}{\pgfqpoint{0.557cm}{1.328cm}}{\pgfqpoint{0.582cm}{1.302cm}}
\pgfpathcurveto{\pgfqpoint{0.608cm}{1.276cm}}{\pgfqpoint{0.643cm}{1.262cm}}{\pgfqpoint{0.679cm}{1.262cm}}
\pgfpathcurveto{\pgfqpoint{0.715cm}{1.262cm}}{\pgfqpoint{0.75cm}{1.276cm}}{\pgfqpoint{0.775cm}{1.302cm}}
\pgfpathcurveto{\pgfqpoint{0.801cm}{1.328cm}}{\pgfqpoint{0.815cm}{1.363cm}}{\pgfqpoint{0.815cm}{1.399cm}}
\pgfusepath{fill}
\pgfpathmoveto{\pgfqpoint{1.345cm}{1.371cm}}
\pgfpathcurveto{\pgfqpoint{1.345cm}{1.408cm}}{\pgfqpoint{1.331cm}{1.442cm}}{\pgfqpoint{1.305cm}{1.468cm}}
\pgfpathcurveto{\pgfqpoint{1.28cm}{1.494cm}}{\pgfqpoint{1.245cm}{1.508cm}}{\pgfqpoint{1.209cm}{1.508cm}}
\pgfpathcurveto{\pgfqpoint{1.172cm}{1.508cm}}{\pgfqpoint{1.138cm}{1.494cm}}{\pgfqpoint{1.112cm}{1.468cm}}
\pgfpathcurveto{\pgfqpoint{1.087cm}{1.442cm}}{\pgfqpoint{1.072cm}{1.408cm}}{\pgfqpoint{1.072cm}{1.371cm}}
\pgfpathcurveto{\pgfqpoint{1.072cm}{1.335cm}}{\pgfqpoint{1.087cm}{1.3cm}}{\pgfqpoint{1.112cm}{1.274cm}}
\pgfpathcurveto{\pgfqpoint{1.138cm}{1.249cm}}{\pgfqpoint{1.172cm}{1.234cm}}{\pgfqpoint{1.209cm}{1.234cm}}
\pgfpathcurveto{\pgfqpoint{1.245cm}{1.234cm}}{\pgfqpoint{1.28cm}{1.249cm}}{\pgfqpoint{1.305cm}{1.274cm}}
\pgfpathcurveto{\pgfqpoint{1.331cm}{1.3cm}}{\pgfqpoint{1.345cm}{1.335cm}}{\pgfqpoint{1.345cm}{1.371cm}}
\pgfusepath{fill}
\begin{pgfscope}
\pgfsetdash{}{0cm}
\pgfsetlinewidth{0.818mm}
\pgfsetroundcap
\pgfsetmiterlimit{4.0}
\pgfpathmoveto{\pgfqpoint{0.682cm}{0.671cm}}
\pgfpathlineto{\pgfqpoint{0.682cm}{0.042cm}}
\pgfusepath{stroke}
\end{pgfscope}
\end{pgfscope}
\end{pgfscope}
\end{pgfscope}
\end{tikzpicture}}},X)}+\rmb{3X\circ (X^{\!\resizebox{0.6em}{!}{
\begin{tikzpicture}
\pgfpathmoveto{\pgfqpoint{0cm}{-0.035cm}}
\pgfpathlineto{\pgfqpoint{1.376cm}{-0.035cm}}
\pgfpathlineto{\pgfqpoint{1.376cm}{1.552cm}}
\pgfpathlineto{\pgfqpoint{0cm}{1.552cm}}
\pgfpathclose
\pgfusepath{clip}
\begin{pgfscope}
\begin{pgfscope}
\pgfpathmoveto{\pgfqpoint{0cm}{-0.035cm}}
\pgfpathlineto{\pgfqpoint{1.376cm}{-0.035cm}}
\pgfpathlineto{\pgfqpoint{1.376cm}{1.552cm}}
\pgfpathlineto{\pgfqpoint{0cm}{1.552cm}}
\pgfpathclose
\pgfusepath{clip}
\begin{pgfscope}
\begin{pgfscope}
\pgfsetdash{}{0cm}
\pgfsetlinewidth{0.818mm}
\pgfsetroundcap
\pgfsetroundjoin
\pgfsetmiterlimit{7.0}
\definecolor{eps2pgf_color}{gray}{0}\pgfsetstrokecolor{eps2pgf_color}\pgfsetfillcolor{eps2pgf_color}
\pgfpathmoveto{\pgfqpoint{0.117cm}{1.421cm}}
\pgfpathlineto{\pgfqpoint{0.682cm}{0.671cm}}
\pgfpathlineto{\pgfqpoint{1.246cm}{1.421cm}}
\pgfusepath{stroke}
\end{pgfscope}
\definecolor{eps2pgf_color}{gray}{0}\pgfsetstrokecolor{eps2pgf_color}\pgfsetfillcolor{eps2pgf_color}
\pgfpathmoveto{\pgfqpoint{0.273cm}{1.395cm}}
\pgfpathcurveto{\pgfqpoint{0.273cm}{1.432cm}}{\pgfqpoint{0.259cm}{1.467cm}}{\pgfqpoint{0.233cm}{1.492cm}}
\pgfpathcurveto{\pgfqpoint{0.207cm}{1.518cm}}{\pgfqpoint{0.173cm}{1.532cm}}{\pgfqpoint{0.137cm}{1.532cm}}
\pgfpathcurveto{\pgfqpoint{0.1cm}{1.532cm}}{\pgfqpoint{0.066cm}{1.518cm}}{\pgfqpoint{0.04cm}{1.492cm}}
\pgfpathcurveto{\pgfqpoint{0.014cm}{1.467cm}}{\pgfqpoint{0cm}{1.432cm}}{\pgfqpoint{0cm}{1.395cm}}
\pgfpathcurveto{\pgfqpoint{0cm}{1.359cm}}{\pgfqpoint{0.014cm}{1.324cm}}{\pgfqpoint{0.04cm}{1.299cm}}
\pgfpathcurveto{\pgfqpoint{0.066cm}{1.273cm}}{\pgfqpoint{0.1cm}{1.258cm}}{\pgfqpoint{0.137cm}{1.258cm}}
\pgfpathcurveto{\pgfqpoint{0.173cm}{1.258cm}}{\pgfqpoint{0.207cm}{1.273cm}}{\pgfqpoint{0.233cm}{1.299cm}}
\pgfpathcurveto{\pgfqpoint{0.259cm}{1.324cm}}{\pgfqpoint{0.273cm}{1.359cm}}{\pgfqpoint{0.273cm}{1.395cm}}
\pgfusepath{fill}
\begin{pgfscope}
\pgfsetdash{}{0cm}
\pgfsetlinewidth{0.818mm}
\pgfsetmiterlimit{7.0}
\pgfpathmoveto{\pgfqpoint{0.682cm}{0.671cm}}
\pgfpathlineto{\pgfqpoint{0.679cm}{1.418cm}}
\pgfusepath{stroke}
\end{pgfscope}
\pgfpathmoveto{\pgfqpoint{0.815cm}{1.399cm}}
\pgfpathcurveto{\pgfqpoint{0.815cm}{1.435cm}}{\pgfqpoint{0.801cm}{1.47cm}}{\pgfqpoint{0.775cm}{1.496cm}}
\pgfpathcurveto{\pgfqpoint{0.75cm}{1.521cm}}{\pgfqpoint{0.715cm}{1.536cm}}{\pgfqpoint{0.679cm}{1.536cm}}
\pgfpathcurveto{\pgfqpoint{0.643cm}{1.536cm}}{\pgfqpoint{0.608cm}{1.521cm}}{\pgfqpoint{0.582cm}{1.496cm}}
\pgfpathcurveto{\pgfqpoint{0.557cm}{1.47cm}}{\pgfqpoint{0.542cm}{1.435cm}}{\pgfqpoint{0.542cm}{1.399cm}}
\pgfpathcurveto{\pgfqpoint{0.542cm}{1.363cm}}{\pgfqpoint{0.557cm}{1.328cm}}{\pgfqpoint{0.582cm}{1.302cm}}
\pgfpathcurveto{\pgfqpoint{0.608cm}{1.276cm}}{\pgfqpoint{0.643cm}{1.262cm}}{\pgfqpoint{0.679cm}{1.262cm}}
\pgfpathcurveto{\pgfqpoint{0.715cm}{1.262cm}}{\pgfqpoint{0.75cm}{1.276cm}}{\pgfqpoint{0.775cm}{1.302cm}}
\pgfpathcurveto{\pgfqpoint{0.801cm}{1.328cm}}{\pgfqpoint{0.815cm}{1.363cm}}{\pgfqpoint{0.815cm}{1.399cm}}
\pgfusepath{fill}
\pgfpathmoveto{\pgfqpoint{1.345cm}{1.371cm}}
\pgfpathcurveto{\pgfqpoint{1.345cm}{1.408cm}}{\pgfqpoint{1.331cm}{1.442cm}}{\pgfqpoint{1.305cm}{1.468cm}}
\pgfpathcurveto{\pgfqpoint{1.28cm}{1.494cm}}{\pgfqpoint{1.245cm}{1.508cm}}{\pgfqpoint{1.209cm}{1.508cm}}
\pgfpathcurveto{\pgfqpoint{1.172cm}{1.508cm}}{\pgfqpoint{1.138cm}{1.494cm}}{\pgfqpoint{1.112cm}{1.468cm}}
\pgfpathcurveto{\pgfqpoint{1.087cm}{1.442cm}}{\pgfqpoint{1.072cm}{1.408cm}}{\pgfqpoint{1.072cm}{1.371cm}}
\pgfpathcurveto{\pgfqpoint{1.072cm}{1.335cm}}{\pgfqpoint{1.087cm}{1.3cm}}{\pgfqpoint{1.112cm}{1.274cm}}
\pgfpathcurveto{\pgfqpoint{1.138cm}{1.249cm}}{\pgfqpoint{1.172cm}{1.234cm}}{\pgfqpoint{1.209cm}{1.234cm}}
\pgfpathcurveto{\pgfqpoint{1.245cm}{1.234cm}}{\pgfqpoint{1.28cm}{1.249cm}}{\pgfqpoint{1.305cm}{1.274cm}}
\pgfpathcurveto{\pgfqpoint{1.331cm}{1.3cm}}{\pgfqpoint{1.345cm}{1.335cm}}{\pgfqpoint{1.345cm}{1.371cm}}
\pgfusepath{fill}
\begin{pgfscope}
\pgfsetdash{}{0cm}
\pgfsetlinewidth{0.818mm}
\pgfsetroundcap
\pgfsetmiterlimit{4.0}
\pgfpathmoveto{\pgfqpoint{0.682cm}{0.671cm}}
\pgfpathlineto{\pgfqpoint{0.682cm}{0.042cm}}
\pgfusepath{stroke}
\end{pgfscope}
\end{pgfscope}
\end{pgfscope}
\end{pgfscope}
\end{tikzpicture}}}\circ X^{\!\resizebox{0.6em}{!}{
\begin{tikzpicture}
\pgfpathmoveto{\pgfqpoint{0cm}{-0.035cm}}
\pgfpathlineto{\pgfqpoint{1.376cm}{-0.035cm}}
\pgfpathlineto{\pgfqpoint{1.376cm}{1.552cm}}
\pgfpathlineto{\pgfqpoint{0cm}{1.552cm}}
\pgfpathclose
\pgfusepath{clip}
\begin{pgfscope}
\begin{pgfscope}
\pgfpathmoveto{\pgfqpoint{0cm}{-0.035cm}}
\pgfpathlineto{\pgfqpoint{1.376cm}{-0.035cm}}
\pgfpathlineto{\pgfqpoint{1.376cm}{1.552cm}}
\pgfpathlineto{\pgfqpoint{0cm}{1.552cm}}
\pgfpathclose
\pgfusepath{clip}
\begin{pgfscope}
\begin{pgfscope}
\pgfsetdash{}{0cm}
\pgfsetlinewidth{0.818mm}
\pgfsetroundcap
\pgfsetroundjoin
\pgfsetmiterlimit{7.0}
\definecolor{eps2pgf_color}{gray}{0}\pgfsetstrokecolor{eps2pgf_color}\pgfsetfillcolor{eps2pgf_color}
\pgfpathmoveto{\pgfqpoint{0.117cm}{1.421cm}}
\pgfpathlineto{\pgfqpoint{0.682cm}{0.671cm}}
\pgfpathlineto{\pgfqpoint{1.246cm}{1.421cm}}
\pgfusepath{stroke}
\end{pgfscope}
\definecolor{eps2pgf_color}{gray}{0}\pgfsetstrokecolor{eps2pgf_color}\pgfsetfillcolor{eps2pgf_color}
\pgfpathmoveto{\pgfqpoint{0.273cm}{1.395cm}}
\pgfpathcurveto{\pgfqpoint{0.273cm}{1.432cm}}{\pgfqpoint{0.259cm}{1.467cm}}{\pgfqpoint{0.233cm}{1.492cm}}
\pgfpathcurveto{\pgfqpoint{0.207cm}{1.518cm}}{\pgfqpoint{0.173cm}{1.532cm}}{\pgfqpoint{0.137cm}{1.532cm}}
\pgfpathcurveto{\pgfqpoint{0.1cm}{1.532cm}}{\pgfqpoint{0.066cm}{1.518cm}}{\pgfqpoint{0.04cm}{1.492cm}}
\pgfpathcurveto{\pgfqpoint{0.014cm}{1.467cm}}{\pgfqpoint{0cm}{1.432cm}}{\pgfqpoint{0cm}{1.395cm}}
\pgfpathcurveto{\pgfqpoint{0cm}{1.359cm}}{\pgfqpoint{0.014cm}{1.324cm}}{\pgfqpoint{0.04cm}{1.299cm}}
\pgfpathcurveto{\pgfqpoint{0.066cm}{1.273cm}}{\pgfqpoint{0.1cm}{1.258cm}}{\pgfqpoint{0.137cm}{1.258cm}}
\pgfpathcurveto{\pgfqpoint{0.173cm}{1.258cm}}{\pgfqpoint{0.207cm}{1.273cm}}{\pgfqpoint{0.233cm}{1.299cm}}
\pgfpathcurveto{\pgfqpoint{0.259cm}{1.324cm}}{\pgfqpoint{0.273cm}{1.359cm}}{\pgfqpoint{0.273cm}{1.395cm}}
\pgfusepath{fill}
\begin{pgfscope}
\pgfsetdash{}{0cm}
\pgfsetlinewidth{0.818mm}
\pgfsetmiterlimit{7.0}
\pgfpathmoveto{\pgfqpoint{0.682cm}{0.671cm}}
\pgfpathlineto{\pgfqpoint{0.679cm}{1.418cm}}
\pgfusepath{stroke}
\end{pgfscope}
\pgfpathmoveto{\pgfqpoint{0.815cm}{1.399cm}}
\pgfpathcurveto{\pgfqpoint{0.815cm}{1.435cm}}{\pgfqpoint{0.801cm}{1.47cm}}{\pgfqpoint{0.775cm}{1.496cm}}
\pgfpathcurveto{\pgfqpoint{0.75cm}{1.521cm}}{\pgfqpoint{0.715cm}{1.536cm}}{\pgfqpoint{0.679cm}{1.536cm}}
\pgfpathcurveto{\pgfqpoint{0.643cm}{1.536cm}}{\pgfqpoint{0.608cm}{1.521cm}}{\pgfqpoint{0.582cm}{1.496cm}}
\pgfpathcurveto{\pgfqpoint{0.557cm}{1.47cm}}{\pgfqpoint{0.542cm}{1.435cm}}{\pgfqpoint{0.542cm}{1.399cm}}
\pgfpathcurveto{\pgfqpoint{0.542cm}{1.363cm}}{\pgfqpoint{0.557cm}{1.328cm}}{\pgfqpoint{0.582cm}{1.302cm}}
\pgfpathcurveto{\pgfqpoint{0.608cm}{1.276cm}}{\pgfqpoint{0.643cm}{1.262cm}}{\pgfqpoint{0.679cm}{1.262cm}}
\pgfpathcurveto{\pgfqpoint{0.715cm}{1.262cm}}{\pgfqpoint{0.75cm}{1.276cm}}{\pgfqpoint{0.775cm}{1.302cm}}
\pgfpathcurveto{\pgfqpoint{0.801cm}{1.328cm}}{\pgfqpoint{0.815cm}{1.363cm}}{\pgfqpoint{0.815cm}{1.399cm}}
\pgfusepath{fill}
\pgfpathmoveto{\pgfqpoint{1.345cm}{1.371cm}}
\pgfpathcurveto{\pgfqpoint{1.345cm}{1.408cm}}{\pgfqpoint{1.331cm}{1.442cm}}{\pgfqpoint{1.305cm}{1.468cm}}
\pgfpathcurveto{\pgfqpoint{1.28cm}{1.494cm}}{\pgfqpoint{1.245cm}{1.508cm}}{\pgfqpoint{1.209cm}{1.508cm}}
\pgfpathcurveto{\pgfqpoint{1.172cm}{1.508cm}}{\pgfqpoint{1.138cm}{1.494cm}}{\pgfqpoint{1.112cm}{1.468cm}}
\pgfpathcurveto{\pgfqpoint{1.087cm}{1.442cm}}{\pgfqpoint{1.072cm}{1.408cm}}{\pgfqpoint{1.072cm}{1.371cm}}
\pgfpathcurveto{\pgfqpoint{1.072cm}{1.335cm}}{\pgfqpoint{1.087cm}{1.3cm}}{\pgfqpoint{1.112cm}{1.274cm}}
\pgfpathcurveto{\pgfqpoint{1.138cm}{1.249cm}}{\pgfqpoint{1.172cm}{1.234cm}}{\pgfqpoint{1.209cm}{1.234cm}}
\pgfpathcurveto{\pgfqpoint{1.245cm}{1.234cm}}{\pgfqpoint{1.28cm}{1.249cm}}{\pgfqpoint{1.305cm}{1.274cm}}
\pgfpathcurveto{\pgfqpoint{1.331cm}{1.3cm}}{\pgfqpoint{1.345cm}{1.335cm}}{\pgfqpoint{1.345cm}{1.371cm}}
\pgfusepath{fill}
\begin{pgfscope}
\pgfsetdash{}{0cm}
\pgfsetlinewidth{0.818mm}
\pgfsetroundcap
\pgfsetmiterlimit{4.0}
\pgfpathmoveto{\pgfqpoint{0.682cm}{0.671cm}}
\pgfpathlineto{\pgfqpoint{0.682cm}{0.042cm}}
\pgfusepath{stroke}
\end{pgfscope}
\end{pgfscope}
\end{pgfscope}
\end{pgfscope}
\end{tikzpicture}}})}\\
&\quad-\rmg{6X\succ(X^{\!\resizebox{0.6em}{!}{
\begin{tikzpicture}
\pgfpathmoveto{\pgfqpoint{0cm}{-0.035cm}}
\pgfpathlineto{\pgfqpoint{1.376cm}{-0.035cm}}
\pgfpathlineto{\pgfqpoint{1.376cm}{1.552cm}}
\pgfpathlineto{\pgfqpoint{0cm}{1.552cm}}
\pgfpathclose
\pgfusepath{clip}
\begin{pgfscope}
\begin{pgfscope}
\pgfpathmoveto{\pgfqpoint{0cm}{-0.035cm}}
\pgfpathlineto{\pgfqpoint{1.376cm}{-0.035cm}}
\pgfpathlineto{\pgfqpoint{1.376cm}{1.552cm}}
\pgfpathlineto{\pgfqpoint{0cm}{1.552cm}}
\pgfpathclose
\pgfusepath{clip}
\begin{pgfscope}
\begin{pgfscope}
\pgfsetdash{}{0cm}
\pgfsetlinewidth{0.818mm}
\pgfsetroundcap
\pgfsetroundjoin
\pgfsetmiterlimit{7.0}
\definecolor{eps2pgf_color}{gray}{0}\pgfsetstrokecolor{eps2pgf_color}\pgfsetfillcolor{eps2pgf_color}
\pgfpathmoveto{\pgfqpoint{0.117cm}{1.421cm}}
\pgfpathlineto{\pgfqpoint{0.682cm}{0.671cm}}
\pgfpathlineto{\pgfqpoint{1.246cm}{1.421cm}}
\pgfusepath{stroke}
\end{pgfscope}
\definecolor{eps2pgf_color}{gray}{0}\pgfsetstrokecolor{eps2pgf_color}\pgfsetfillcolor{eps2pgf_color}
\pgfpathmoveto{\pgfqpoint{0.273cm}{1.395cm}}
\pgfpathcurveto{\pgfqpoint{0.273cm}{1.432cm}}{\pgfqpoint{0.259cm}{1.467cm}}{\pgfqpoint{0.233cm}{1.492cm}}
\pgfpathcurveto{\pgfqpoint{0.207cm}{1.518cm}}{\pgfqpoint{0.173cm}{1.532cm}}{\pgfqpoint{0.137cm}{1.532cm}}
\pgfpathcurveto{\pgfqpoint{0.1cm}{1.532cm}}{\pgfqpoint{0.066cm}{1.518cm}}{\pgfqpoint{0.04cm}{1.492cm}}
\pgfpathcurveto{\pgfqpoint{0.014cm}{1.467cm}}{\pgfqpoint{0cm}{1.432cm}}{\pgfqpoint{0cm}{1.395cm}}
\pgfpathcurveto{\pgfqpoint{0cm}{1.359cm}}{\pgfqpoint{0.014cm}{1.324cm}}{\pgfqpoint{0.04cm}{1.299cm}}
\pgfpathcurveto{\pgfqpoint{0.066cm}{1.273cm}}{\pgfqpoint{0.1cm}{1.258cm}}{\pgfqpoint{0.137cm}{1.258cm}}
\pgfpathcurveto{\pgfqpoint{0.173cm}{1.258cm}}{\pgfqpoint{0.207cm}{1.273cm}}{\pgfqpoint{0.233cm}{1.299cm}}
\pgfpathcurveto{\pgfqpoint{0.259cm}{1.324cm}}{\pgfqpoint{0.273cm}{1.359cm}}{\pgfqpoint{0.273cm}{1.395cm}}
\pgfusepath{fill}
\begin{pgfscope}
\pgfsetdash{}{0cm}
\pgfsetlinewidth{0.818mm}
\pgfsetmiterlimit{7.0}
\pgfpathmoveto{\pgfqpoint{0.682cm}{0.671cm}}
\pgfpathlineto{\pgfqpoint{0.679cm}{1.418cm}}
\pgfusepath{stroke}
\end{pgfscope}
\pgfpathmoveto{\pgfqpoint{0.815cm}{1.399cm}}
\pgfpathcurveto{\pgfqpoint{0.815cm}{1.435cm}}{\pgfqpoint{0.801cm}{1.47cm}}{\pgfqpoint{0.775cm}{1.496cm}}
\pgfpathcurveto{\pgfqpoint{0.75cm}{1.521cm}}{\pgfqpoint{0.715cm}{1.536cm}}{\pgfqpoint{0.679cm}{1.536cm}}
\pgfpathcurveto{\pgfqpoint{0.643cm}{1.536cm}}{\pgfqpoint{0.608cm}{1.521cm}}{\pgfqpoint{0.582cm}{1.496cm}}
\pgfpathcurveto{\pgfqpoint{0.557cm}{1.47cm}}{\pgfqpoint{0.542cm}{1.435cm}}{\pgfqpoint{0.542cm}{1.399cm}}
\pgfpathcurveto{\pgfqpoint{0.542cm}{1.363cm}}{\pgfqpoint{0.557cm}{1.328cm}}{\pgfqpoint{0.582cm}{1.302cm}}
\pgfpathcurveto{\pgfqpoint{0.608cm}{1.276cm}}{\pgfqpoint{0.643cm}{1.262cm}}{\pgfqpoint{0.679cm}{1.262cm}}
\pgfpathcurveto{\pgfqpoint{0.715cm}{1.262cm}}{\pgfqpoint{0.75cm}{1.276cm}}{\pgfqpoint{0.775cm}{1.302cm}}
\pgfpathcurveto{\pgfqpoint{0.801cm}{1.328cm}}{\pgfqpoint{0.815cm}{1.363cm}}{\pgfqpoint{0.815cm}{1.399cm}}
\pgfusepath{fill}
\pgfpathmoveto{\pgfqpoint{1.345cm}{1.371cm}}
\pgfpathcurveto{\pgfqpoint{1.345cm}{1.408cm}}{\pgfqpoint{1.331cm}{1.442cm}}{\pgfqpoint{1.305cm}{1.468cm}}
\pgfpathcurveto{\pgfqpoint{1.28cm}{1.494cm}}{\pgfqpoint{1.245cm}{1.508cm}}{\pgfqpoint{1.209cm}{1.508cm}}
\pgfpathcurveto{\pgfqpoint{1.172cm}{1.508cm}}{\pgfqpoint{1.138cm}{1.494cm}}{\pgfqpoint{1.112cm}{1.468cm}}
\pgfpathcurveto{\pgfqpoint{1.087cm}{1.442cm}}{\pgfqpoint{1.072cm}{1.408cm}}{\pgfqpoint{1.072cm}{1.371cm}}
\pgfpathcurveto{\pgfqpoint{1.072cm}{1.335cm}}{\pgfqpoint{1.087cm}{1.3cm}}{\pgfqpoint{1.112cm}{1.274cm}}
\pgfpathcurveto{\pgfqpoint{1.138cm}{1.249cm}}{\pgfqpoint{1.172cm}{1.234cm}}{\pgfqpoint{1.209cm}{1.234cm}}
\pgfpathcurveto{\pgfqpoint{1.245cm}{1.234cm}}{\pgfqpoint{1.28cm}{1.249cm}}{\pgfqpoint{1.305cm}{1.274cm}}
\pgfpathcurveto{\pgfqpoint{1.331cm}{1.3cm}}{\pgfqpoint{1.345cm}{1.335cm}}{\pgfqpoint{1.345cm}{1.371cm}}
\pgfusepath{fill}
\begin{pgfscope}
\pgfsetdash{}{0cm}
\pgfsetlinewidth{0.818mm}
\pgfsetroundcap
\pgfsetmiterlimit{4.0}
\pgfpathmoveto{\pgfqpoint{0.682cm}{0.671cm}}
\pgfpathlineto{\pgfqpoint{0.682cm}{0.042cm}}
\pgfusepath{stroke}
\end{pgfscope}
\end{pgfscope}
\end{pgfscope}
\end{pgfscope}
\end{tikzpicture}}}(\phi+\psi))}-\rmg{6X\prec(X^{\!\resizebox{0.6em}{!}{
\begin{tikzpicture}
\pgfpathmoveto{\pgfqpoint{0cm}{-0.035cm}}
\pgfpathlineto{\pgfqpoint{1.376cm}{-0.035cm}}
\pgfpathlineto{\pgfqpoint{1.376cm}{1.552cm}}
\pgfpathlineto{\pgfqpoint{0cm}{1.552cm}}
\pgfpathclose
\pgfusepath{clip}
\begin{pgfscope}
\begin{pgfscope}
\pgfpathmoveto{\pgfqpoint{0cm}{-0.035cm}}
\pgfpathlineto{\pgfqpoint{1.376cm}{-0.035cm}}
\pgfpathlineto{\pgfqpoint{1.376cm}{1.552cm}}
\pgfpathlineto{\pgfqpoint{0cm}{1.552cm}}
\pgfpathclose
\pgfusepath{clip}
\begin{pgfscope}
\begin{pgfscope}
\pgfsetdash{}{0cm}
\pgfsetlinewidth{0.818mm}
\pgfsetroundcap
\pgfsetroundjoin
\pgfsetmiterlimit{7.0}
\definecolor{eps2pgf_color}{gray}{0}\pgfsetstrokecolor{eps2pgf_color}\pgfsetfillcolor{eps2pgf_color}
\pgfpathmoveto{\pgfqpoint{0.117cm}{1.421cm}}
\pgfpathlineto{\pgfqpoint{0.682cm}{0.671cm}}
\pgfpathlineto{\pgfqpoint{1.246cm}{1.421cm}}
\pgfusepath{stroke}
\end{pgfscope}
\definecolor{eps2pgf_color}{gray}{0}\pgfsetstrokecolor{eps2pgf_color}\pgfsetfillcolor{eps2pgf_color}
\pgfpathmoveto{\pgfqpoint{0.273cm}{1.395cm}}
\pgfpathcurveto{\pgfqpoint{0.273cm}{1.432cm}}{\pgfqpoint{0.259cm}{1.467cm}}{\pgfqpoint{0.233cm}{1.492cm}}
\pgfpathcurveto{\pgfqpoint{0.207cm}{1.518cm}}{\pgfqpoint{0.173cm}{1.532cm}}{\pgfqpoint{0.137cm}{1.532cm}}
\pgfpathcurveto{\pgfqpoint{0.1cm}{1.532cm}}{\pgfqpoint{0.066cm}{1.518cm}}{\pgfqpoint{0.04cm}{1.492cm}}
\pgfpathcurveto{\pgfqpoint{0.014cm}{1.467cm}}{\pgfqpoint{0cm}{1.432cm}}{\pgfqpoint{0cm}{1.395cm}}
\pgfpathcurveto{\pgfqpoint{0cm}{1.359cm}}{\pgfqpoint{0.014cm}{1.324cm}}{\pgfqpoint{0.04cm}{1.299cm}}
\pgfpathcurveto{\pgfqpoint{0.066cm}{1.273cm}}{\pgfqpoint{0.1cm}{1.258cm}}{\pgfqpoint{0.137cm}{1.258cm}}
\pgfpathcurveto{\pgfqpoint{0.173cm}{1.258cm}}{\pgfqpoint{0.207cm}{1.273cm}}{\pgfqpoint{0.233cm}{1.299cm}}
\pgfpathcurveto{\pgfqpoint{0.259cm}{1.324cm}}{\pgfqpoint{0.273cm}{1.359cm}}{\pgfqpoint{0.273cm}{1.395cm}}
\pgfusepath{fill}
\begin{pgfscope}
\pgfsetdash{}{0cm}
\pgfsetlinewidth{0.818mm}
\pgfsetmiterlimit{7.0}
\pgfpathmoveto{\pgfqpoint{0.682cm}{0.671cm}}
\pgfpathlineto{\pgfqpoint{0.679cm}{1.418cm}}
\pgfusepath{stroke}
\end{pgfscope}
\pgfpathmoveto{\pgfqpoint{0.815cm}{1.399cm}}
\pgfpathcurveto{\pgfqpoint{0.815cm}{1.435cm}}{\pgfqpoint{0.801cm}{1.47cm}}{\pgfqpoint{0.775cm}{1.496cm}}
\pgfpathcurveto{\pgfqpoint{0.75cm}{1.521cm}}{\pgfqpoint{0.715cm}{1.536cm}}{\pgfqpoint{0.679cm}{1.536cm}}
\pgfpathcurveto{\pgfqpoint{0.643cm}{1.536cm}}{\pgfqpoint{0.608cm}{1.521cm}}{\pgfqpoint{0.582cm}{1.496cm}}
\pgfpathcurveto{\pgfqpoint{0.557cm}{1.47cm}}{\pgfqpoint{0.542cm}{1.435cm}}{\pgfqpoint{0.542cm}{1.399cm}}
\pgfpathcurveto{\pgfqpoint{0.542cm}{1.363cm}}{\pgfqpoint{0.557cm}{1.328cm}}{\pgfqpoint{0.582cm}{1.302cm}}
\pgfpathcurveto{\pgfqpoint{0.608cm}{1.276cm}}{\pgfqpoint{0.643cm}{1.262cm}}{\pgfqpoint{0.679cm}{1.262cm}}
\pgfpathcurveto{\pgfqpoint{0.715cm}{1.262cm}}{\pgfqpoint{0.75cm}{1.276cm}}{\pgfqpoint{0.775cm}{1.302cm}}
\pgfpathcurveto{\pgfqpoint{0.801cm}{1.328cm}}{\pgfqpoint{0.815cm}{1.363cm}}{\pgfqpoint{0.815cm}{1.399cm}}
\pgfusepath{fill}
\pgfpathmoveto{\pgfqpoint{1.345cm}{1.371cm}}
\pgfpathcurveto{\pgfqpoint{1.345cm}{1.408cm}}{\pgfqpoint{1.331cm}{1.442cm}}{\pgfqpoint{1.305cm}{1.468cm}}
\pgfpathcurveto{\pgfqpoint{1.28cm}{1.494cm}}{\pgfqpoint{1.245cm}{1.508cm}}{\pgfqpoint{1.209cm}{1.508cm}}
\pgfpathcurveto{\pgfqpoint{1.172cm}{1.508cm}}{\pgfqpoint{1.138cm}{1.494cm}}{\pgfqpoint{1.112cm}{1.468cm}}
\pgfpathcurveto{\pgfqpoint{1.087cm}{1.442cm}}{\pgfqpoint{1.072cm}{1.408cm}}{\pgfqpoint{1.072cm}{1.371cm}}
\pgfpathcurveto{\pgfqpoint{1.072cm}{1.335cm}}{\pgfqpoint{1.087cm}{1.3cm}}{\pgfqpoint{1.112cm}{1.274cm}}
\pgfpathcurveto{\pgfqpoint{1.138cm}{1.249cm}}{\pgfqpoint{1.172cm}{1.234cm}}{\pgfqpoint{1.209cm}{1.234cm}}
\pgfpathcurveto{\pgfqpoint{1.245cm}{1.234cm}}{\pgfqpoint{1.28cm}{1.249cm}}{\pgfqpoint{1.305cm}{1.274cm}}
\pgfpathcurveto{\pgfqpoint{1.331cm}{1.3cm}}{\pgfqpoint{1.345cm}{1.335cm}}{\pgfqpoint{1.345cm}{1.371cm}}
\pgfusepath{fill}
\begin{pgfscope}
\pgfsetdash{}{0cm}
\pgfsetlinewidth{0.818mm}
\pgfsetroundcap
\pgfsetmiterlimit{4.0}
\pgfpathmoveto{\pgfqpoint{0.682cm}{0.671cm}}
\pgfpathlineto{\pgfqpoint{0.682cm}{0.042cm}}
\pgfusepath{stroke}
\end{pgfscope}
\end{pgfscope}
\end{pgfscope}
\end{pgfscope}
\end{tikzpicture}}}(\phi+\psi))}-\rmb{6X\circ(X^{\!\resizebox{0.6em}{!}{
\begin{tikzpicture}
\pgfpathmoveto{\pgfqpoint{0cm}{-0.035cm}}
\pgfpathlineto{\pgfqpoint{1.376cm}{-0.035cm}}
\pgfpathlineto{\pgfqpoint{1.376cm}{1.552cm}}
\pgfpathlineto{\pgfqpoint{0cm}{1.552cm}}
\pgfpathclose
\pgfusepath{clip}
\begin{pgfscope}
\begin{pgfscope}
\pgfpathmoveto{\pgfqpoint{0cm}{-0.035cm}}
\pgfpathlineto{\pgfqpoint{1.376cm}{-0.035cm}}
\pgfpathlineto{\pgfqpoint{1.376cm}{1.552cm}}
\pgfpathlineto{\pgfqpoint{0cm}{1.552cm}}
\pgfpathclose
\pgfusepath{clip}
\begin{pgfscope}
\begin{pgfscope}
\pgfsetdash{}{0cm}
\pgfsetlinewidth{0.818mm}
\pgfsetroundcap
\pgfsetroundjoin
\pgfsetmiterlimit{7.0}
\definecolor{eps2pgf_color}{gray}{0}\pgfsetstrokecolor{eps2pgf_color}\pgfsetfillcolor{eps2pgf_color}
\pgfpathmoveto{\pgfqpoint{0.117cm}{1.421cm}}
\pgfpathlineto{\pgfqpoint{0.682cm}{0.671cm}}
\pgfpathlineto{\pgfqpoint{1.246cm}{1.421cm}}
\pgfusepath{stroke}
\end{pgfscope}
\definecolor{eps2pgf_color}{gray}{0}\pgfsetstrokecolor{eps2pgf_color}\pgfsetfillcolor{eps2pgf_color}
\pgfpathmoveto{\pgfqpoint{0.273cm}{1.395cm}}
\pgfpathcurveto{\pgfqpoint{0.273cm}{1.432cm}}{\pgfqpoint{0.259cm}{1.467cm}}{\pgfqpoint{0.233cm}{1.492cm}}
\pgfpathcurveto{\pgfqpoint{0.207cm}{1.518cm}}{\pgfqpoint{0.173cm}{1.532cm}}{\pgfqpoint{0.137cm}{1.532cm}}
\pgfpathcurveto{\pgfqpoint{0.1cm}{1.532cm}}{\pgfqpoint{0.066cm}{1.518cm}}{\pgfqpoint{0.04cm}{1.492cm}}
\pgfpathcurveto{\pgfqpoint{0.014cm}{1.467cm}}{\pgfqpoint{0cm}{1.432cm}}{\pgfqpoint{0cm}{1.395cm}}
\pgfpathcurveto{\pgfqpoint{0cm}{1.359cm}}{\pgfqpoint{0.014cm}{1.324cm}}{\pgfqpoint{0.04cm}{1.299cm}}
\pgfpathcurveto{\pgfqpoint{0.066cm}{1.273cm}}{\pgfqpoint{0.1cm}{1.258cm}}{\pgfqpoint{0.137cm}{1.258cm}}
\pgfpathcurveto{\pgfqpoint{0.173cm}{1.258cm}}{\pgfqpoint{0.207cm}{1.273cm}}{\pgfqpoint{0.233cm}{1.299cm}}
\pgfpathcurveto{\pgfqpoint{0.259cm}{1.324cm}}{\pgfqpoint{0.273cm}{1.359cm}}{\pgfqpoint{0.273cm}{1.395cm}}
\pgfusepath{fill}
\begin{pgfscope}
\pgfsetdash{}{0cm}
\pgfsetlinewidth{0.818mm}
\pgfsetmiterlimit{7.0}
\pgfpathmoveto{\pgfqpoint{0.682cm}{0.671cm}}
\pgfpathlineto{\pgfqpoint{0.679cm}{1.418cm}}
\pgfusepath{stroke}
\end{pgfscope}
\pgfpathmoveto{\pgfqpoint{0.815cm}{1.399cm}}
\pgfpathcurveto{\pgfqpoint{0.815cm}{1.435cm}}{\pgfqpoint{0.801cm}{1.47cm}}{\pgfqpoint{0.775cm}{1.496cm}}
\pgfpathcurveto{\pgfqpoint{0.75cm}{1.521cm}}{\pgfqpoint{0.715cm}{1.536cm}}{\pgfqpoint{0.679cm}{1.536cm}}
\pgfpathcurveto{\pgfqpoint{0.643cm}{1.536cm}}{\pgfqpoint{0.608cm}{1.521cm}}{\pgfqpoint{0.582cm}{1.496cm}}
\pgfpathcurveto{\pgfqpoint{0.557cm}{1.47cm}}{\pgfqpoint{0.542cm}{1.435cm}}{\pgfqpoint{0.542cm}{1.399cm}}
\pgfpathcurveto{\pgfqpoint{0.542cm}{1.363cm}}{\pgfqpoint{0.557cm}{1.328cm}}{\pgfqpoint{0.582cm}{1.302cm}}
\pgfpathcurveto{\pgfqpoint{0.608cm}{1.276cm}}{\pgfqpoint{0.643cm}{1.262cm}}{\pgfqpoint{0.679cm}{1.262cm}}
\pgfpathcurveto{\pgfqpoint{0.715cm}{1.262cm}}{\pgfqpoint{0.75cm}{1.276cm}}{\pgfqpoint{0.775cm}{1.302cm}}
\pgfpathcurveto{\pgfqpoint{0.801cm}{1.328cm}}{\pgfqpoint{0.815cm}{1.363cm}}{\pgfqpoint{0.815cm}{1.399cm}}
\pgfusepath{fill}
\pgfpathmoveto{\pgfqpoint{1.345cm}{1.371cm}}
\pgfpathcurveto{\pgfqpoint{1.345cm}{1.408cm}}{\pgfqpoint{1.331cm}{1.442cm}}{\pgfqpoint{1.305cm}{1.468cm}}
\pgfpathcurveto{\pgfqpoint{1.28cm}{1.494cm}}{\pgfqpoint{1.245cm}{1.508cm}}{\pgfqpoint{1.209cm}{1.508cm}}
\pgfpathcurveto{\pgfqpoint{1.172cm}{1.508cm}}{\pgfqpoint{1.138cm}{1.494cm}}{\pgfqpoint{1.112cm}{1.468cm}}
\pgfpathcurveto{\pgfqpoint{1.087cm}{1.442cm}}{\pgfqpoint{1.072cm}{1.408cm}}{\pgfqpoint{1.072cm}{1.371cm}}
\pgfpathcurveto{\pgfqpoint{1.072cm}{1.335cm}}{\pgfqpoint{1.087cm}{1.3cm}}{\pgfqpoint{1.112cm}{1.274cm}}
\pgfpathcurveto{\pgfqpoint{1.138cm}{1.249cm}}{\pgfqpoint{1.172cm}{1.234cm}}{\pgfqpoint{1.209cm}{1.234cm}}
\pgfpathcurveto{\pgfqpoint{1.245cm}{1.234cm}}{\pgfqpoint{1.28cm}{1.249cm}}{\pgfqpoint{1.305cm}{1.274cm}}
\pgfpathcurveto{\pgfqpoint{1.331cm}{1.3cm}}{\pgfqpoint{1.345cm}{1.335cm}}{\pgfqpoint{1.345cm}{1.371cm}}
\pgfusepath{fill}
\begin{pgfscope}
\pgfsetdash{}{0cm}
\pgfsetlinewidth{0.818mm}
\pgfsetroundcap
\pgfsetmiterlimit{4.0}
\pgfpathmoveto{\pgfqpoint{0.682cm}{0.671cm}}
\pgfpathlineto{\pgfqpoint{0.682cm}{0.042cm}}
\pgfusepath{stroke}
\end{pgfscope}
\end{pgfscope}
\end{pgfscope}
\end{pgfscope}
\end{tikzpicture}}}\preccurlyeq(\phi+\psi))}\\
&\quad-\rmg{6(\phi+\psi)X^{\!\resizebox{!}{.8em}{
\begin{tikzpicture}
\pgfpathmoveto{\pgfqpoint{0cm}{-0.035cm}}
\pgfpathlineto{\pgfqpoint{1.976cm}{-0.035cm}}
\pgfpathlineto{\pgfqpoint{1.976cm}{1.94cm}}
\pgfpathlineto{\pgfqpoint{0cm}{1.94cm}}
\pgfpathclose
\pgfusepath{clip}
\begin{pgfscope}
\begin{pgfscope}
\pgfpathmoveto{\pgfqpoint{0cm}{-0.035cm}}
\pgfpathlineto{\pgfqpoint{1.976cm}{-0.035cm}}
\pgfpathlineto{\pgfqpoint{1.976cm}{1.94cm}}
\pgfpathlineto{\pgfqpoint{0cm}{1.94cm}}
\pgfpathclose
\pgfusepath{clip}
\begin{pgfscope}
\begin{pgfscope}
\pgfsetdash{}{0cm}
\pgfsetlinewidth{0.818mm}
\pgfsetroundcap
\pgfsetroundjoin
\pgfsetmiterlimit{7.0}
\definecolor{eps2pgf_color}{gray}{0}\pgfsetstrokecolor{eps2pgf_color}\pgfsetfillcolor{eps2pgf_color}
\pgfpathmoveto{\pgfqpoint{0.117cm}{1.815cm}}
\pgfpathlineto{\pgfqpoint{0.682cm}{1.065cm}}
\pgfpathlineto{\pgfqpoint{1.246cm}{1.815cm}}
\pgfusepath{stroke}
\end{pgfscope}
\definecolor{eps2pgf_color}{gray}{0}\pgfsetstrokecolor{eps2pgf_color}\pgfsetfillcolor{eps2pgf_color}
\pgfpathmoveto{\pgfqpoint{0.273cm}{1.789cm}}
\pgfpathcurveto{\pgfqpoint{0.273cm}{1.825cm}}{\pgfqpoint{0.259cm}{1.86cm}}{\pgfqpoint{0.233cm}{1.886cm}}
\pgfpathcurveto{\pgfqpoint{0.207cm}{1.912cm}}{\pgfqpoint{0.173cm}{1.926cm}}{\pgfqpoint{0.137cm}{1.926cm}}
\pgfpathcurveto{\pgfqpoint{0.1cm}{1.926cm}}{\pgfqpoint{0.066cm}{1.912cm}}{\pgfqpoint{0.04cm}{1.886cm}}
\pgfpathcurveto{\pgfqpoint{0.014cm}{1.86cm}}{\pgfqpoint{0cm}{1.825cm}}{\pgfqpoint{0cm}{1.789cm}}
\pgfpathcurveto{\pgfqpoint{0cm}{1.753cm}}{\pgfqpoint{0.014cm}{1.718cm}}{\pgfqpoint{0.04cm}{1.692cm}}
\pgfpathcurveto{\pgfqpoint{0.066cm}{1.667cm}}{\pgfqpoint{0.1cm}{1.652cm}}{\pgfqpoint{0.137cm}{1.652cm}}
\pgfpathcurveto{\pgfqpoint{0.173cm}{1.652cm}}{\pgfqpoint{0.207cm}{1.667cm}}{\pgfqpoint{0.233cm}{1.692cm}}
\pgfpathcurveto{\pgfqpoint{0.259cm}{1.718cm}}{\pgfqpoint{0.273cm}{1.753cm}}{\pgfqpoint{0.273cm}{1.789cm}}
\pgfusepath{fill}
\begin{pgfscope}
\pgfsetdash{}{0cm}
\pgfsetlinewidth{0.818mm}
\pgfsetmiterlimit{7.0}
\pgfpathmoveto{\pgfqpoint{0.682cm}{1.065cm}}
\pgfpathlineto{\pgfqpoint{0.679cm}{1.812cm}}
\pgfusepath{stroke}
\end{pgfscope}
\pgfpathmoveto{\pgfqpoint{0.815cm}{1.793cm}}
\pgfpathcurveto{\pgfqpoint{0.815cm}{1.829cm}}{\pgfqpoint{0.801cm}{1.864cm}}{\pgfqpoint{0.775cm}{1.89cm}}
\pgfpathcurveto{\pgfqpoint{0.75cm}{1.915cm}}{\pgfqpoint{0.715cm}{1.93cm}}{\pgfqpoint{0.679cm}{1.93cm}}
\pgfpathcurveto{\pgfqpoint{0.643cm}{1.93cm}}{\pgfqpoint{0.608cm}{1.915cm}}{\pgfqpoint{0.582cm}{1.89cm}}
\pgfpathcurveto{\pgfqpoint{0.557cm}{1.864cm}}{\pgfqpoint{0.542cm}{1.829cm}}{\pgfqpoint{0.542cm}{1.793cm}}
\pgfpathcurveto{\pgfqpoint{0.542cm}{1.756cm}}{\pgfqpoint{0.557cm}{1.722cm}}{\pgfqpoint{0.582cm}{1.696cm}}
\pgfpathcurveto{\pgfqpoint{0.608cm}{1.67cm}}{\pgfqpoint{0.643cm}{1.656cm}}{\pgfqpoint{0.679cm}{1.656cm}}
\pgfpathcurveto{\pgfqpoint{0.715cm}{1.656cm}}{\pgfqpoint{0.75cm}{1.67cm}}{\pgfqpoint{0.775cm}{1.696cm}}
\pgfpathcurveto{\pgfqpoint{0.801cm}{1.722cm}}{\pgfqpoint{0.815cm}{1.756cm}}{\pgfqpoint{0.815cm}{1.793cm}}
\pgfusepath{fill}
\pgfpathmoveto{\pgfqpoint{1.345cm}{1.765cm}}
\pgfpathcurveto{\pgfqpoint{1.345cm}{1.801cm}}{\pgfqpoint{1.331cm}{1.836cm}}{\pgfqpoint{1.305cm}{1.862cm}}
\pgfpathcurveto{\pgfqpoint{1.28cm}{1.887cm}}{\pgfqpoint{1.245cm}{1.902cm}}{\pgfqpoint{1.209cm}{1.902cm}}
\pgfpathcurveto{\pgfqpoint{1.172cm}{1.902cm}}{\pgfqpoint{1.138cm}{1.887cm}}{\pgfqpoint{1.112cm}{1.862cm}}
\pgfpathcurveto{\pgfqpoint{1.087cm}{1.836cm}}{\pgfqpoint{1.072cm}{1.801cm}}{\pgfqpoint{1.072cm}{1.765cm}}
\pgfpathcurveto{\pgfqpoint{1.072cm}{1.728cm}}{\pgfqpoint{1.087cm}{1.694cm}}{\pgfqpoint{1.112cm}{1.668cm}}
\pgfpathcurveto{\pgfqpoint{1.138cm}{1.642cm}}{\pgfqpoint{1.172cm}{1.628cm}}{\pgfqpoint{1.209cm}{1.628cm}}
\pgfpathcurveto{\pgfqpoint{1.245cm}{1.628cm}}{\pgfqpoint{1.28cm}{1.642cm}}{\pgfqpoint{1.305cm}{1.668cm}}
\pgfpathcurveto{\pgfqpoint{1.331cm}{1.694cm}}{\pgfqpoint{1.345cm}{1.728cm}}{\pgfqpoint{1.345cm}{1.765cm}}
\pgfusepath{fill}
\begin{pgfscope}
\pgfsetdash{}{0cm}
\pgfsetlinewidth{0.818mm}
\pgfsetroundcap
\pgfsetroundjoin
\pgfsetmiterlimit{7.0}
\pgfpathmoveto{\pgfqpoint{0.682cm}{1.065cm}}
\pgfpathlineto{\pgfqpoint{1.246cm}{0.315cm}}
\pgfpathlineto{\pgfqpoint{1.811cm}{1.065cm}}
\pgfusepath{stroke}
\end{pgfscope}
\pgfpathmoveto{\pgfqpoint{1.948cm}{1.065cm}}
\pgfpathcurveto{\pgfqpoint{1.948cm}{1.101cm}}{\pgfqpoint{1.933cm}{1.136cm}}{\pgfqpoint{1.907cm}{1.162cm}}
\pgfpathcurveto{\pgfqpoint{1.882cm}{1.187cm}}{\pgfqpoint{1.847cm}{1.202cm}}{\pgfqpoint{1.811cm}{1.202cm}}
\pgfpathcurveto{\pgfqpoint{1.775cm}{1.202cm}}{\pgfqpoint{1.74cm}{1.187cm}}{\pgfqpoint{1.714cm}{1.162cm}}
\pgfpathcurveto{\pgfqpoint{1.689cm}{1.136cm}}{\pgfqpoint{1.674cm}{1.101cm}}{\pgfqpoint{1.674cm}{1.065cm}}
\pgfpathcurveto{\pgfqpoint{1.674cm}{1.029cm}}{\pgfqpoint{1.689cm}{0.994cm}}{\pgfqpoint{1.714cm}{0.968cm}}
\pgfpathcurveto{\pgfqpoint{1.74cm}{0.942cm}}{\pgfqpoint{1.775cm}{0.928cm}}{\pgfqpoint{1.811cm}{0.928cm}}
\pgfpathcurveto{\pgfqpoint{1.847cm}{0.928cm}}{\pgfqpoint{1.882cm}{0.942cm}}{\pgfqpoint{1.907cm}{0.968cm}}
\pgfpathcurveto{\pgfqpoint{1.933cm}{0.994cm}}{\pgfqpoint{1.948cm}{1.029cm}}{\pgfqpoint{1.948cm}{1.065cm}}
\pgfusepath{fill}
\begin{pgfscope}
\pgfsetdash{}{0cm}
\pgfsetlinewidth{0.818mm}
\pgfsetmiterlimit{4.0}
\pgfpathmoveto{\pgfqpoint{1.383cm}{0.178cm}}
\pgfpathcurveto{\pgfqpoint{1.383cm}{0.214cm}}{\pgfqpoint{1.369cm}{0.249cm}}{\pgfqpoint{1.343cm}{0.275cm}}
\pgfpathcurveto{\pgfqpoint{1.317cm}{0.3cm}}{\pgfqpoint{1.283cm}{0.315cm}}{\pgfqpoint{1.246cm}{0.315cm}}
\pgfpathcurveto{\pgfqpoint{1.21cm}{0.315cm}}{\pgfqpoint{1.175cm}{0.3cm}}{\pgfqpoint{1.15cm}{0.275cm}}
\pgfpathcurveto{\pgfqpoint{1.124cm}{0.249cm}}{\pgfqpoint{1.11cm}{0.214cm}}{\pgfqpoint{1.11cm}{0.178cm}}
\pgfpathcurveto{\pgfqpoint{1.11cm}{0.141cm}}{\pgfqpoint{1.124cm}{0.107cm}}{\pgfqpoint{1.15cm}{0.081cm}}
\pgfpathcurveto{\pgfqpoint{1.175cm}{0.055cm}}{\pgfqpoint{1.21cm}{0.041cm}}{\pgfqpoint{1.246cm}{0.041cm}}
\pgfpathcurveto{\pgfqpoint{1.283cm}{0.041cm}}{\pgfqpoint{1.317cm}{0.055cm}}{\pgfqpoint{1.343cm}{0.081cm}}
\pgfpathcurveto{\pgfqpoint{1.369cm}{0.107cm}}{\pgfqpoint{1.383cm}{0.141cm}}{\pgfqpoint{1.383cm}{0.178cm}}
\pgfusepath{stroke}
\end{pgfscope}
\end{pgfscope}
\end{pgfscope}
\end{pgfscope}
\end{tikzpicture}}}}-\rmb{6\mathrm{com}(\phi+\psi,X^{\!\resizebox{0.6em}{!}{
\begin{tikzpicture}
\pgfpathmoveto{\pgfqpoint{0cm}{-0.035cm}}
\pgfpathlineto{\pgfqpoint{1.376cm}{-0.035cm}}
\pgfpathlineto{\pgfqpoint{1.376cm}{1.552cm}}
\pgfpathlineto{\pgfqpoint{0cm}{1.552cm}}
\pgfpathclose
\pgfusepath{clip}
\begin{pgfscope}
\begin{pgfscope}
\pgfpathmoveto{\pgfqpoint{0cm}{-0.035cm}}
\pgfpathlineto{\pgfqpoint{1.376cm}{-0.035cm}}
\pgfpathlineto{\pgfqpoint{1.376cm}{1.552cm}}
\pgfpathlineto{\pgfqpoint{0cm}{1.552cm}}
\pgfpathclose
\pgfusepath{clip}
\begin{pgfscope}
\begin{pgfscope}
\pgfsetdash{}{0cm}
\pgfsetlinewidth{0.818mm}
\pgfsetroundcap
\pgfsetroundjoin
\pgfsetmiterlimit{7.0}
\definecolor{eps2pgf_color}{gray}{0}\pgfsetstrokecolor{eps2pgf_color}\pgfsetfillcolor{eps2pgf_color}
\pgfpathmoveto{\pgfqpoint{0.117cm}{1.421cm}}
\pgfpathlineto{\pgfqpoint{0.682cm}{0.671cm}}
\pgfpathlineto{\pgfqpoint{1.246cm}{1.421cm}}
\pgfusepath{stroke}
\end{pgfscope}
\definecolor{eps2pgf_color}{gray}{0}\pgfsetstrokecolor{eps2pgf_color}\pgfsetfillcolor{eps2pgf_color}
\pgfpathmoveto{\pgfqpoint{0.273cm}{1.395cm}}
\pgfpathcurveto{\pgfqpoint{0.273cm}{1.432cm}}{\pgfqpoint{0.259cm}{1.467cm}}{\pgfqpoint{0.233cm}{1.492cm}}
\pgfpathcurveto{\pgfqpoint{0.207cm}{1.518cm}}{\pgfqpoint{0.173cm}{1.532cm}}{\pgfqpoint{0.137cm}{1.532cm}}
\pgfpathcurveto{\pgfqpoint{0.1cm}{1.532cm}}{\pgfqpoint{0.066cm}{1.518cm}}{\pgfqpoint{0.04cm}{1.492cm}}
\pgfpathcurveto{\pgfqpoint{0.014cm}{1.467cm}}{\pgfqpoint{0cm}{1.432cm}}{\pgfqpoint{0cm}{1.395cm}}
\pgfpathcurveto{\pgfqpoint{0cm}{1.359cm}}{\pgfqpoint{0.014cm}{1.324cm}}{\pgfqpoint{0.04cm}{1.299cm}}
\pgfpathcurveto{\pgfqpoint{0.066cm}{1.273cm}}{\pgfqpoint{0.1cm}{1.258cm}}{\pgfqpoint{0.137cm}{1.258cm}}
\pgfpathcurveto{\pgfqpoint{0.173cm}{1.258cm}}{\pgfqpoint{0.207cm}{1.273cm}}{\pgfqpoint{0.233cm}{1.299cm}}
\pgfpathcurveto{\pgfqpoint{0.259cm}{1.324cm}}{\pgfqpoint{0.273cm}{1.359cm}}{\pgfqpoint{0.273cm}{1.395cm}}
\pgfusepath{fill}
\begin{pgfscope}
\pgfsetdash{}{0cm}
\pgfsetlinewidth{0.818mm}
\pgfsetmiterlimit{7.0}
\pgfpathmoveto{\pgfqpoint{0.682cm}{0.671cm}}
\pgfpathlineto{\pgfqpoint{0.679cm}{1.418cm}}
\pgfusepath{stroke}
\end{pgfscope}
\pgfpathmoveto{\pgfqpoint{0.815cm}{1.399cm}}
\pgfpathcurveto{\pgfqpoint{0.815cm}{1.435cm}}{\pgfqpoint{0.801cm}{1.47cm}}{\pgfqpoint{0.775cm}{1.496cm}}
\pgfpathcurveto{\pgfqpoint{0.75cm}{1.521cm}}{\pgfqpoint{0.715cm}{1.536cm}}{\pgfqpoint{0.679cm}{1.536cm}}
\pgfpathcurveto{\pgfqpoint{0.643cm}{1.536cm}}{\pgfqpoint{0.608cm}{1.521cm}}{\pgfqpoint{0.582cm}{1.496cm}}
\pgfpathcurveto{\pgfqpoint{0.557cm}{1.47cm}}{\pgfqpoint{0.542cm}{1.435cm}}{\pgfqpoint{0.542cm}{1.399cm}}
\pgfpathcurveto{\pgfqpoint{0.542cm}{1.363cm}}{\pgfqpoint{0.557cm}{1.328cm}}{\pgfqpoint{0.582cm}{1.302cm}}
\pgfpathcurveto{\pgfqpoint{0.608cm}{1.276cm}}{\pgfqpoint{0.643cm}{1.262cm}}{\pgfqpoint{0.679cm}{1.262cm}}
\pgfpathcurveto{\pgfqpoint{0.715cm}{1.262cm}}{\pgfqpoint{0.75cm}{1.276cm}}{\pgfqpoint{0.775cm}{1.302cm}}
\pgfpathcurveto{\pgfqpoint{0.801cm}{1.328cm}}{\pgfqpoint{0.815cm}{1.363cm}}{\pgfqpoint{0.815cm}{1.399cm}}
\pgfusepath{fill}
\pgfpathmoveto{\pgfqpoint{1.345cm}{1.371cm}}
\pgfpathcurveto{\pgfqpoint{1.345cm}{1.408cm}}{\pgfqpoint{1.331cm}{1.442cm}}{\pgfqpoint{1.305cm}{1.468cm}}
\pgfpathcurveto{\pgfqpoint{1.28cm}{1.494cm}}{\pgfqpoint{1.245cm}{1.508cm}}{\pgfqpoint{1.209cm}{1.508cm}}
\pgfpathcurveto{\pgfqpoint{1.172cm}{1.508cm}}{\pgfqpoint{1.138cm}{1.494cm}}{\pgfqpoint{1.112cm}{1.468cm}}
\pgfpathcurveto{\pgfqpoint{1.087cm}{1.442cm}}{\pgfqpoint{1.072cm}{1.408cm}}{\pgfqpoint{1.072cm}{1.371cm}}
\pgfpathcurveto{\pgfqpoint{1.072cm}{1.335cm}}{\pgfqpoint{1.087cm}{1.3cm}}{\pgfqpoint{1.112cm}{1.274cm}}
\pgfpathcurveto{\pgfqpoint{1.138cm}{1.249cm}}{\pgfqpoint{1.172cm}{1.234cm}}{\pgfqpoint{1.209cm}{1.234cm}}
\pgfpathcurveto{\pgfqpoint{1.245cm}{1.234cm}}{\pgfqpoint{1.28cm}{1.249cm}}{\pgfqpoint{1.305cm}{1.274cm}}
\pgfpathcurveto{\pgfqpoint{1.331cm}{1.3cm}}{\pgfqpoint{1.345cm}{1.335cm}}{\pgfqpoint{1.345cm}{1.371cm}}
\pgfusepath{fill}
\begin{pgfscope}
\pgfsetdash{}{0cm}
\pgfsetlinewidth{0.818mm}
\pgfsetroundcap
\pgfsetmiterlimit{4.0}
\pgfpathmoveto{\pgfqpoint{0.682cm}{0.671cm}}
\pgfpathlineto{\pgfqpoint{0.682cm}{0.042cm}}
\pgfusepath{stroke}
\end{pgfscope}
\end{pgfscope}
\end{pgfscope}
\end{pgfscope}
\end{tikzpicture}}},X)}+\rmg{3X(\phi+\psi)^2}.
\end{align*}
For the remaining term in \eqref{eq:rhs45a} we write
$$
(-X^{\!\resizebox{0.6em}{!}{
\begin{tikzpicture}
\pgfpathmoveto{\pgfqpoint{0cm}{-0.035cm}}
\pgfpathlineto{\pgfqpoint{1.376cm}{-0.035cm}}
\pgfpathlineto{\pgfqpoint{1.376cm}{1.552cm}}
\pgfpathlineto{\pgfqpoint{0cm}{1.552cm}}
\pgfpathclose
\pgfusepath{clip}
\begin{pgfscope}
\begin{pgfscope}
\pgfpathmoveto{\pgfqpoint{0cm}{-0.035cm}}
\pgfpathlineto{\pgfqpoint{1.376cm}{-0.035cm}}
\pgfpathlineto{\pgfqpoint{1.376cm}{1.552cm}}
\pgfpathlineto{\pgfqpoint{0cm}{1.552cm}}
\pgfpathclose
\pgfusepath{clip}
\begin{pgfscope}
\begin{pgfscope}
\pgfsetdash{}{0cm}
\pgfsetlinewidth{0.818mm}
\pgfsetroundcap
\pgfsetroundjoin
\pgfsetmiterlimit{7.0}
\definecolor{eps2pgf_color}{gray}{0}\pgfsetstrokecolor{eps2pgf_color}\pgfsetfillcolor{eps2pgf_color}
\pgfpathmoveto{\pgfqpoint{0.117cm}{1.421cm}}
\pgfpathlineto{\pgfqpoint{0.682cm}{0.671cm}}
\pgfpathlineto{\pgfqpoint{1.246cm}{1.421cm}}
\pgfusepath{stroke}
\end{pgfscope}
\definecolor{eps2pgf_color}{gray}{0}\pgfsetstrokecolor{eps2pgf_color}\pgfsetfillcolor{eps2pgf_color}
\pgfpathmoveto{\pgfqpoint{0.273cm}{1.395cm}}
\pgfpathcurveto{\pgfqpoint{0.273cm}{1.432cm}}{\pgfqpoint{0.259cm}{1.467cm}}{\pgfqpoint{0.233cm}{1.492cm}}
\pgfpathcurveto{\pgfqpoint{0.207cm}{1.518cm}}{\pgfqpoint{0.173cm}{1.532cm}}{\pgfqpoint{0.137cm}{1.532cm}}
\pgfpathcurveto{\pgfqpoint{0.1cm}{1.532cm}}{\pgfqpoint{0.066cm}{1.518cm}}{\pgfqpoint{0.04cm}{1.492cm}}
\pgfpathcurveto{\pgfqpoint{0.014cm}{1.467cm}}{\pgfqpoint{0cm}{1.432cm}}{\pgfqpoint{0cm}{1.395cm}}
\pgfpathcurveto{\pgfqpoint{0cm}{1.359cm}}{\pgfqpoint{0.014cm}{1.324cm}}{\pgfqpoint{0.04cm}{1.299cm}}
\pgfpathcurveto{\pgfqpoint{0.066cm}{1.273cm}}{\pgfqpoint{0.1cm}{1.258cm}}{\pgfqpoint{0.137cm}{1.258cm}}
\pgfpathcurveto{\pgfqpoint{0.173cm}{1.258cm}}{\pgfqpoint{0.207cm}{1.273cm}}{\pgfqpoint{0.233cm}{1.299cm}}
\pgfpathcurveto{\pgfqpoint{0.259cm}{1.324cm}}{\pgfqpoint{0.273cm}{1.359cm}}{\pgfqpoint{0.273cm}{1.395cm}}
\pgfusepath{fill}
\begin{pgfscope}
\pgfsetdash{}{0cm}
\pgfsetlinewidth{0.818mm}
\pgfsetmiterlimit{7.0}
\pgfpathmoveto{\pgfqpoint{0.682cm}{0.671cm}}
\pgfpathlineto{\pgfqpoint{0.679cm}{1.418cm}}
\pgfusepath{stroke}
\end{pgfscope}
\pgfpathmoveto{\pgfqpoint{0.815cm}{1.399cm}}
\pgfpathcurveto{\pgfqpoint{0.815cm}{1.435cm}}{\pgfqpoint{0.801cm}{1.47cm}}{\pgfqpoint{0.775cm}{1.496cm}}
\pgfpathcurveto{\pgfqpoint{0.75cm}{1.521cm}}{\pgfqpoint{0.715cm}{1.536cm}}{\pgfqpoint{0.679cm}{1.536cm}}
\pgfpathcurveto{\pgfqpoint{0.643cm}{1.536cm}}{\pgfqpoint{0.608cm}{1.521cm}}{\pgfqpoint{0.582cm}{1.496cm}}
\pgfpathcurveto{\pgfqpoint{0.557cm}{1.47cm}}{\pgfqpoint{0.542cm}{1.435cm}}{\pgfqpoint{0.542cm}{1.399cm}}
\pgfpathcurveto{\pgfqpoint{0.542cm}{1.363cm}}{\pgfqpoint{0.557cm}{1.328cm}}{\pgfqpoint{0.582cm}{1.302cm}}
\pgfpathcurveto{\pgfqpoint{0.608cm}{1.276cm}}{\pgfqpoint{0.643cm}{1.262cm}}{\pgfqpoint{0.679cm}{1.262cm}}
\pgfpathcurveto{\pgfqpoint{0.715cm}{1.262cm}}{\pgfqpoint{0.75cm}{1.276cm}}{\pgfqpoint{0.775cm}{1.302cm}}
\pgfpathcurveto{\pgfqpoint{0.801cm}{1.328cm}}{\pgfqpoint{0.815cm}{1.363cm}}{\pgfqpoint{0.815cm}{1.399cm}}
\pgfusepath{fill}
\pgfpathmoveto{\pgfqpoint{1.345cm}{1.371cm}}
\pgfpathcurveto{\pgfqpoint{1.345cm}{1.408cm}}{\pgfqpoint{1.331cm}{1.442cm}}{\pgfqpoint{1.305cm}{1.468cm}}
\pgfpathcurveto{\pgfqpoint{1.28cm}{1.494cm}}{\pgfqpoint{1.245cm}{1.508cm}}{\pgfqpoint{1.209cm}{1.508cm}}
\pgfpathcurveto{\pgfqpoint{1.172cm}{1.508cm}}{\pgfqpoint{1.138cm}{1.494cm}}{\pgfqpoint{1.112cm}{1.468cm}}
\pgfpathcurveto{\pgfqpoint{1.087cm}{1.442cm}}{\pgfqpoint{1.072cm}{1.408cm}}{\pgfqpoint{1.072cm}{1.371cm}}
\pgfpathcurveto{\pgfqpoint{1.072cm}{1.335cm}}{\pgfqpoint{1.087cm}{1.3cm}}{\pgfqpoint{1.112cm}{1.274cm}}
\pgfpathcurveto{\pgfqpoint{1.138cm}{1.249cm}}{\pgfqpoint{1.172cm}{1.234cm}}{\pgfqpoint{1.209cm}{1.234cm}}
\pgfpathcurveto{\pgfqpoint{1.245cm}{1.234cm}}{\pgfqpoint{1.28cm}{1.249cm}}{\pgfqpoint{1.305cm}{1.274cm}}
\pgfpathcurveto{\pgfqpoint{1.331cm}{1.3cm}}{\pgfqpoint{1.345cm}{1.335cm}}{\pgfqpoint{1.345cm}{1.371cm}}
\pgfusepath{fill}
\begin{pgfscope}
\pgfsetdash{}{0cm}
\pgfsetlinewidth{0.818mm}
\pgfsetroundcap
\pgfsetmiterlimit{4.0}
\pgfpathmoveto{\pgfqpoint{0.682cm}{0.671cm}}
\pgfpathlineto{\pgfqpoint{0.682cm}{0.042cm}}
\pgfusepath{stroke}
\end{pgfscope}
\end{pgfscope}
\end{pgfscope}
\end{pgfscope}
\end{tikzpicture}}}+\phi+\psi)^3=\rmb{(-X^{\!\resizebox{0.6em}{!}{
\begin{tikzpicture}
\pgfpathmoveto{\pgfqpoint{0cm}{-0.035cm}}
\pgfpathlineto{\pgfqpoint{1.376cm}{-0.035cm}}
\pgfpathlineto{\pgfqpoint{1.376cm}{1.552cm}}
\pgfpathlineto{\pgfqpoint{0cm}{1.552cm}}
\pgfpathclose
\pgfusepath{clip}
\begin{pgfscope}
\begin{pgfscope}
\pgfpathmoveto{\pgfqpoint{0cm}{-0.035cm}}
\pgfpathlineto{\pgfqpoint{1.376cm}{-0.035cm}}
\pgfpathlineto{\pgfqpoint{1.376cm}{1.552cm}}
\pgfpathlineto{\pgfqpoint{0cm}{1.552cm}}
\pgfpathclose
\pgfusepath{clip}
\begin{pgfscope}
\begin{pgfscope}
\pgfsetdash{}{0cm}
\pgfsetlinewidth{0.818mm}
\pgfsetroundcap
\pgfsetroundjoin
\pgfsetmiterlimit{7.0}
\definecolor{eps2pgf_color}{gray}{0}\pgfsetstrokecolor{eps2pgf_color}\pgfsetfillcolor{eps2pgf_color}
\pgfpathmoveto{\pgfqpoint{0.117cm}{1.421cm}}
\pgfpathlineto{\pgfqpoint{0.682cm}{0.671cm}}
\pgfpathlineto{\pgfqpoint{1.246cm}{1.421cm}}
\pgfusepath{stroke}
\end{pgfscope}
\definecolor{eps2pgf_color}{gray}{0}\pgfsetstrokecolor{eps2pgf_color}\pgfsetfillcolor{eps2pgf_color}
\pgfpathmoveto{\pgfqpoint{0.273cm}{1.395cm}}
\pgfpathcurveto{\pgfqpoint{0.273cm}{1.432cm}}{\pgfqpoint{0.259cm}{1.467cm}}{\pgfqpoint{0.233cm}{1.492cm}}
\pgfpathcurveto{\pgfqpoint{0.207cm}{1.518cm}}{\pgfqpoint{0.173cm}{1.532cm}}{\pgfqpoint{0.137cm}{1.532cm}}
\pgfpathcurveto{\pgfqpoint{0.1cm}{1.532cm}}{\pgfqpoint{0.066cm}{1.518cm}}{\pgfqpoint{0.04cm}{1.492cm}}
\pgfpathcurveto{\pgfqpoint{0.014cm}{1.467cm}}{\pgfqpoint{0cm}{1.432cm}}{\pgfqpoint{0cm}{1.395cm}}
\pgfpathcurveto{\pgfqpoint{0cm}{1.359cm}}{\pgfqpoint{0.014cm}{1.324cm}}{\pgfqpoint{0.04cm}{1.299cm}}
\pgfpathcurveto{\pgfqpoint{0.066cm}{1.273cm}}{\pgfqpoint{0.1cm}{1.258cm}}{\pgfqpoint{0.137cm}{1.258cm}}
\pgfpathcurveto{\pgfqpoint{0.173cm}{1.258cm}}{\pgfqpoint{0.207cm}{1.273cm}}{\pgfqpoint{0.233cm}{1.299cm}}
\pgfpathcurveto{\pgfqpoint{0.259cm}{1.324cm}}{\pgfqpoint{0.273cm}{1.359cm}}{\pgfqpoint{0.273cm}{1.395cm}}
\pgfusepath{fill}
\begin{pgfscope}
\pgfsetdash{}{0cm}
\pgfsetlinewidth{0.818mm}
\pgfsetmiterlimit{7.0}
\pgfpathmoveto{\pgfqpoint{0.682cm}{0.671cm}}
\pgfpathlineto{\pgfqpoint{0.679cm}{1.418cm}}
\pgfusepath{stroke}
\end{pgfscope}
\pgfpathmoveto{\pgfqpoint{0.815cm}{1.399cm}}
\pgfpathcurveto{\pgfqpoint{0.815cm}{1.435cm}}{\pgfqpoint{0.801cm}{1.47cm}}{\pgfqpoint{0.775cm}{1.496cm}}
\pgfpathcurveto{\pgfqpoint{0.75cm}{1.521cm}}{\pgfqpoint{0.715cm}{1.536cm}}{\pgfqpoint{0.679cm}{1.536cm}}
\pgfpathcurveto{\pgfqpoint{0.643cm}{1.536cm}}{\pgfqpoint{0.608cm}{1.521cm}}{\pgfqpoint{0.582cm}{1.496cm}}
\pgfpathcurveto{\pgfqpoint{0.557cm}{1.47cm}}{\pgfqpoint{0.542cm}{1.435cm}}{\pgfqpoint{0.542cm}{1.399cm}}
\pgfpathcurveto{\pgfqpoint{0.542cm}{1.363cm}}{\pgfqpoint{0.557cm}{1.328cm}}{\pgfqpoint{0.582cm}{1.302cm}}
\pgfpathcurveto{\pgfqpoint{0.608cm}{1.276cm}}{\pgfqpoint{0.643cm}{1.262cm}}{\pgfqpoint{0.679cm}{1.262cm}}
\pgfpathcurveto{\pgfqpoint{0.715cm}{1.262cm}}{\pgfqpoint{0.75cm}{1.276cm}}{\pgfqpoint{0.775cm}{1.302cm}}
\pgfpathcurveto{\pgfqpoint{0.801cm}{1.328cm}}{\pgfqpoint{0.815cm}{1.363cm}}{\pgfqpoint{0.815cm}{1.399cm}}
\pgfusepath{fill}
\pgfpathmoveto{\pgfqpoint{1.345cm}{1.371cm}}
\pgfpathcurveto{\pgfqpoint{1.345cm}{1.408cm}}{\pgfqpoint{1.331cm}{1.442cm}}{\pgfqpoint{1.305cm}{1.468cm}}
\pgfpathcurveto{\pgfqpoint{1.28cm}{1.494cm}}{\pgfqpoint{1.245cm}{1.508cm}}{\pgfqpoint{1.209cm}{1.508cm}}
\pgfpathcurveto{\pgfqpoint{1.172cm}{1.508cm}}{\pgfqpoint{1.138cm}{1.494cm}}{\pgfqpoint{1.112cm}{1.468cm}}
\pgfpathcurveto{\pgfqpoint{1.087cm}{1.442cm}}{\pgfqpoint{1.072cm}{1.408cm}}{\pgfqpoint{1.072cm}{1.371cm}}
\pgfpathcurveto{\pgfqpoint{1.072cm}{1.335cm}}{\pgfqpoint{1.087cm}{1.3cm}}{\pgfqpoint{1.112cm}{1.274cm}}
\pgfpathcurveto{\pgfqpoint{1.138cm}{1.249cm}}{\pgfqpoint{1.172cm}{1.234cm}}{\pgfqpoint{1.209cm}{1.234cm}}
\pgfpathcurveto{\pgfqpoint{1.245cm}{1.234cm}}{\pgfqpoint{1.28cm}{1.249cm}}{\pgfqpoint{1.305cm}{1.274cm}}
\pgfpathcurveto{\pgfqpoint{1.331cm}{1.3cm}}{\pgfqpoint{1.345cm}{1.335cm}}{\pgfqpoint{1.345cm}{1.371cm}}
\pgfusepath{fill}
\begin{pgfscope}
\pgfsetdash{}{0cm}
\pgfsetlinewidth{0.818mm}
\pgfsetroundcap
\pgfsetmiterlimit{4.0}
\pgfpathmoveto{\pgfqpoint{0.682cm}{0.671cm}}
\pgfpathlineto{\pgfqpoint{0.682cm}{0.042cm}}
\pgfusepath{stroke}
\end{pgfscope}
\end{pgfscope}
\end{pgfscope}
\end{pgfscope}
\end{tikzpicture}}}+\phi)^3}+\rmb{3(-X^{\!\resizebox{0.6em}{!}{
\begin{tikzpicture}
\pgfpathmoveto{\pgfqpoint{0cm}{-0.035cm}}
\pgfpathlineto{\pgfqpoint{1.376cm}{-0.035cm}}
\pgfpathlineto{\pgfqpoint{1.376cm}{1.552cm}}
\pgfpathlineto{\pgfqpoint{0cm}{1.552cm}}
\pgfpathclose
\pgfusepath{clip}
\begin{pgfscope}
\begin{pgfscope}
\pgfpathmoveto{\pgfqpoint{0cm}{-0.035cm}}
\pgfpathlineto{\pgfqpoint{1.376cm}{-0.035cm}}
\pgfpathlineto{\pgfqpoint{1.376cm}{1.552cm}}
\pgfpathlineto{\pgfqpoint{0cm}{1.552cm}}
\pgfpathclose
\pgfusepath{clip}
\begin{pgfscope}
\begin{pgfscope}
\pgfsetdash{}{0cm}
\pgfsetlinewidth{0.818mm}
\pgfsetroundcap
\pgfsetroundjoin
\pgfsetmiterlimit{7.0}
\definecolor{eps2pgf_color}{gray}{0}\pgfsetstrokecolor{eps2pgf_color}\pgfsetfillcolor{eps2pgf_color}
\pgfpathmoveto{\pgfqpoint{0.117cm}{1.421cm}}
\pgfpathlineto{\pgfqpoint{0.682cm}{0.671cm}}
\pgfpathlineto{\pgfqpoint{1.246cm}{1.421cm}}
\pgfusepath{stroke}
\end{pgfscope}
\definecolor{eps2pgf_color}{gray}{0}\pgfsetstrokecolor{eps2pgf_color}\pgfsetfillcolor{eps2pgf_color}
\pgfpathmoveto{\pgfqpoint{0.273cm}{1.395cm}}
\pgfpathcurveto{\pgfqpoint{0.273cm}{1.432cm}}{\pgfqpoint{0.259cm}{1.467cm}}{\pgfqpoint{0.233cm}{1.492cm}}
\pgfpathcurveto{\pgfqpoint{0.207cm}{1.518cm}}{\pgfqpoint{0.173cm}{1.532cm}}{\pgfqpoint{0.137cm}{1.532cm}}
\pgfpathcurveto{\pgfqpoint{0.1cm}{1.532cm}}{\pgfqpoint{0.066cm}{1.518cm}}{\pgfqpoint{0.04cm}{1.492cm}}
\pgfpathcurveto{\pgfqpoint{0.014cm}{1.467cm}}{\pgfqpoint{0cm}{1.432cm}}{\pgfqpoint{0cm}{1.395cm}}
\pgfpathcurveto{\pgfqpoint{0cm}{1.359cm}}{\pgfqpoint{0.014cm}{1.324cm}}{\pgfqpoint{0.04cm}{1.299cm}}
\pgfpathcurveto{\pgfqpoint{0.066cm}{1.273cm}}{\pgfqpoint{0.1cm}{1.258cm}}{\pgfqpoint{0.137cm}{1.258cm}}
\pgfpathcurveto{\pgfqpoint{0.173cm}{1.258cm}}{\pgfqpoint{0.207cm}{1.273cm}}{\pgfqpoint{0.233cm}{1.299cm}}
\pgfpathcurveto{\pgfqpoint{0.259cm}{1.324cm}}{\pgfqpoint{0.273cm}{1.359cm}}{\pgfqpoint{0.273cm}{1.395cm}}
\pgfusepath{fill}
\begin{pgfscope}
\pgfsetdash{}{0cm}
\pgfsetlinewidth{0.818mm}
\pgfsetmiterlimit{7.0}
\pgfpathmoveto{\pgfqpoint{0.682cm}{0.671cm}}
\pgfpathlineto{\pgfqpoint{0.679cm}{1.418cm}}
\pgfusepath{stroke}
\end{pgfscope}
\pgfpathmoveto{\pgfqpoint{0.815cm}{1.399cm}}
\pgfpathcurveto{\pgfqpoint{0.815cm}{1.435cm}}{\pgfqpoint{0.801cm}{1.47cm}}{\pgfqpoint{0.775cm}{1.496cm}}
\pgfpathcurveto{\pgfqpoint{0.75cm}{1.521cm}}{\pgfqpoint{0.715cm}{1.536cm}}{\pgfqpoint{0.679cm}{1.536cm}}
\pgfpathcurveto{\pgfqpoint{0.643cm}{1.536cm}}{\pgfqpoint{0.608cm}{1.521cm}}{\pgfqpoint{0.582cm}{1.496cm}}
\pgfpathcurveto{\pgfqpoint{0.557cm}{1.47cm}}{\pgfqpoint{0.542cm}{1.435cm}}{\pgfqpoint{0.542cm}{1.399cm}}
\pgfpathcurveto{\pgfqpoint{0.542cm}{1.363cm}}{\pgfqpoint{0.557cm}{1.328cm}}{\pgfqpoint{0.582cm}{1.302cm}}
\pgfpathcurveto{\pgfqpoint{0.608cm}{1.276cm}}{\pgfqpoint{0.643cm}{1.262cm}}{\pgfqpoint{0.679cm}{1.262cm}}
\pgfpathcurveto{\pgfqpoint{0.715cm}{1.262cm}}{\pgfqpoint{0.75cm}{1.276cm}}{\pgfqpoint{0.775cm}{1.302cm}}
\pgfpathcurveto{\pgfqpoint{0.801cm}{1.328cm}}{\pgfqpoint{0.815cm}{1.363cm}}{\pgfqpoint{0.815cm}{1.399cm}}
\pgfusepath{fill}
\pgfpathmoveto{\pgfqpoint{1.345cm}{1.371cm}}
\pgfpathcurveto{\pgfqpoint{1.345cm}{1.408cm}}{\pgfqpoint{1.331cm}{1.442cm}}{\pgfqpoint{1.305cm}{1.468cm}}
\pgfpathcurveto{\pgfqpoint{1.28cm}{1.494cm}}{\pgfqpoint{1.245cm}{1.508cm}}{\pgfqpoint{1.209cm}{1.508cm}}
\pgfpathcurveto{\pgfqpoint{1.172cm}{1.508cm}}{\pgfqpoint{1.138cm}{1.494cm}}{\pgfqpoint{1.112cm}{1.468cm}}
\pgfpathcurveto{\pgfqpoint{1.087cm}{1.442cm}}{\pgfqpoint{1.072cm}{1.408cm}}{\pgfqpoint{1.072cm}{1.371cm}}
\pgfpathcurveto{\pgfqpoint{1.072cm}{1.335cm}}{\pgfqpoint{1.087cm}{1.3cm}}{\pgfqpoint{1.112cm}{1.274cm}}
\pgfpathcurveto{\pgfqpoint{1.138cm}{1.249cm}}{\pgfqpoint{1.172cm}{1.234cm}}{\pgfqpoint{1.209cm}{1.234cm}}
\pgfpathcurveto{\pgfqpoint{1.245cm}{1.234cm}}{\pgfqpoint{1.28cm}{1.249cm}}{\pgfqpoint{1.305cm}{1.274cm}}
\pgfpathcurveto{\pgfqpoint{1.331cm}{1.3cm}}{\pgfqpoint{1.345cm}{1.335cm}}{\pgfqpoint{1.345cm}{1.371cm}}
\pgfusepath{fill}
\begin{pgfscope}
\pgfsetdash{}{0cm}
\pgfsetlinewidth{0.818mm}
\pgfsetroundcap
\pgfsetmiterlimit{4.0}
\pgfpathmoveto{\pgfqpoint{0.682cm}{0.671cm}}
\pgfpathlineto{\pgfqpoint{0.682cm}{0.042cm}}
\pgfusepath{stroke}
\end{pgfscope}
\end{pgfscope}
\end{pgfscope}
\end{pgfscope}
\end{tikzpicture}}}+\phi)^2\psi}+\rmb{3(-X^{\!\resizebox{0.6em}{!}{
\begin{tikzpicture}
\pgfpathmoveto{\pgfqpoint{0cm}{-0.035cm}}
\pgfpathlineto{\pgfqpoint{1.376cm}{-0.035cm}}
\pgfpathlineto{\pgfqpoint{1.376cm}{1.552cm}}
\pgfpathlineto{\pgfqpoint{0cm}{1.552cm}}
\pgfpathclose
\pgfusepath{clip}
\begin{pgfscope}
\begin{pgfscope}
\pgfpathmoveto{\pgfqpoint{0cm}{-0.035cm}}
\pgfpathlineto{\pgfqpoint{1.376cm}{-0.035cm}}
\pgfpathlineto{\pgfqpoint{1.376cm}{1.552cm}}
\pgfpathlineto{\pgfqpoint{0cm}{1.552cm}}
\pgfpathclose
\pgfusepath{clip}
\begin{pgfscope}
\begin{pgfscope}
\pgfsetdash{}{0cm}
\pgfsetlinewidth{0.818mm}
\pgfsetroundcap
\pgfsetroundjoin
\pgfsetmiterlimit{7.0}
\definecolor{eps2pgf_color}{gray}{0}\pgfsetstrokecolor{eps2pgf_color}\pgfsetfillcolor{eps2pgf_color}
\pgfpathmoveto{\pgfqpoint{0.117cm}{1.421cm}}
\pgfpathlineto{\pgfqpoint{0.682cm}{0.671cm}}
\pgfpathlineto{\pgfqpoint{1.246cm}{1.421cm}}
\pgfusepath{stroke}
\end{pgfscope}
\definecolor{eps2pgf_color}{gray}{0}\pgfsetstrokecolor{eps2pgf_color}\pgfsetfillcolor{eps2pgf_color}
\pgfpathmoveto{\pgfqpoint{0.273cm}{1.395cm}}
\pgfpathcurveto{\pgfqpoint{0.273cm}{1.432cm}}{\pgfqpoint{0.259cm}{1.467cm}}{\pgfqpoint{0.233cm}{1.492cm}}
\pgfpathcurveto{\pgfqpoint{0.207cm}{1.518cm}}{\pgfqpoint{0.173cm}{1.532cm}}{\pgfqpoint{0.137cm}{1.532cm}}
\pgfpathcurveto{\pgfqpoint{0.1cm}{1.532cm}}{\pgfqpoint{0.066cm}{1.518cm}}{\pgfqpoint{0.04cm}{1.492cm}}
\pgfpathcurveto{\pgfqpoint{0.014cm}{1.467cm}}{\pgfqpoint{0cm}{1.432cm}}{\pgfqpoint{0cm}{1.395cm}}
\pgfpathcurveto{\pgfqpoint{0cm}{1.359cm}}{\pgfqpoint{0.014cm}{1.324cm}}{\pgfqpoint{0.04cm}{1.299cm}}
\pgfpathcurveto{\pgfqpoint{0.066cm}{1.273cm}}{\pgfqpoint{0.1cm}{1.258cm}}{\pgfqpoint{0.137cm}{1.258cm}}
\pgfpathcurveto{\pgfqpoint{0.173cm}{1.258cm}}{\pgfqpoint{0.207cm}{1.273cm}}{\pgfqpoint{0.233cm}{1.299cm}}
\pgfpathcurveto{\pgfqpoint{0.259cm}{1.324cm}}{\pgfqpoint{0.273cm}{1.359cm}}{\pgfqpoint{0.273cm}{1.395cm}}
\pgfusepath{fill}
\begin{pgfscope}
\pgfsetdash{}{0cm}
\pgfsetlinewidth{0.818mm}
\pgfsetmiterlimit{7.0}
\pgfpathmoveto{\pgfqpoint{0.682cm}{0.671cm}}
\pgfpathlineto{\pgfqpoint{0.679cm}{1.418cm}}
\pgfusepath{stroke}
\end{pgfscope}
\pgfpathmoveto{\pgfqpoint{0.815cm}{1.399cm}}
\pgfpathcurveto{\pgfqpoint{0.815cm}{1.435cm}}{\pgfqpoint{0.801cm}{1.47cm}}{\pgfqpoint{0.775cm}{1.496cm}}
\pgfpathcurveto{\pgfqpoint{0.75cm}{1.521cm}}{\pgfqpoint{0.715cm}{1.536cm}}{\pgfqpoint{0.679cm}{1.536cm}}
\pgfpathcurveto{\pgfqpoint{0.643cm}{1.536cm}}{\pgfqpoint{0.608cm}{1.521cm}}{\pgfqpoint{0.582cm}{1.496cm}}
\pgfpathcurveto{\pgfqpoint{0.557cm}{1.47cm}}{\pgfqpoint{0.542cm}{1.435cm}}{\pgfqpoint{0.542cm}{1.399cm}}
\pgfpathcurveto{\pgfqpoint{0.542cm}{1.363cm}}{\pgfqpoint{0.557cm}{1.328cm}}{\pgfqpoint{0.582cm}{1.302cm}}
\pgfpathcurveto{\pgfqpoint{0.608cm}{1.276cm}}{\pgfqpoint{0.643cm}{1.262cm}}{\pgfqpoint{0.679cm}{1.262cm}}
\pgfpathcurveto{\pgfqpoint{0.715cm}{1.262cm}}{\pgfqpoint{0.75cm}{1.276cm}}{\pgfqpoint{0.775cm}{1.302cm}}
\pgfpathcurveto{\pgfqpoint{0.801cm}{1.328cm}}{\pgfqpoint{0.815cm}{1.363cm}}{\pgfqpoint{0.815cm}{1.399cm}}
\pgfusepath{fill}
\pgfpathmoveto{\pgfqpoint{1.345cm}{1.371cm}}
\pgfpathcurveto{\pgfqpoint{1.345cm}{1.408cm}}{\pgfqpoint{1.331cm}{1.442cm}}{\pgfqpoint{1.305cm}{1.468cm}}
\pgfpathcurveto{\pgfqpoint{1.28cm}{1.494cm}}{\pgfqpoint{1.245cm}{1.508cm}}{\pgfqpoint{1.209cm}{1.508cm}}
\pgfpathcurveto{\pgfqpoint{1.172cm}{1.508cm}}{\pgfqpoint{1.138cm}{1.494cm}}{\pgfqpoint{1.112cm}{1.468cm}}
\pgfpathcurveto{\pgfqpoint{1.087cm}{1.442cm}}{\pgfqpoint{1.072cm}{1.408cm}}{\pgfqpoint{1.072cm}{1.371cm}}
\pgfpathcurveto{\pgfqpoint{1.072cm}{1.335cm}}{\pgfqpoint{1.087cm}{1.3cm}}{\pgfqpoint{1.112cm}{1.274cm}}
\pgfpathcurveto{\pgfqpoint{1.138cm}{1.249cm}}{\pgfqpoint{1.172cm}{1.234cm}}{\pgfqpoint{1.209cm}{1.234cm}}
\pgfpathcurveto{\pgfqpoint{1.245cm}{1.234cm}}{\pgfqpoint{1.28cm}{1.249cm}}{\pgfqpoint{1.305cm}{1.274cm}}
\pgfpathcurveto{\pgfqpoint{1.331cm}{1.3cm}}{\pgfqpoint{1.345cm}{1.335cm}}{\pgfqpoint{1.345cm}{1.371cm}}
\pgfusepath{fill}
\begin{pgfscope}
\pgfsetdash{}{0cm}
\pgfsetlinewidth{0.818mm}
\pgfsetroundcap
\pgfsetmiterlimit{4.0}
\pgfpathmoveto{\pgfqpoint{0.682cm}{0.671cm}}
\pgfpathlineto{\pgfqpoint{0.682cm}{0.042cm}}
\pgfusepath{stroke}
\end{pgfscope}
\end{pgfscope}
\end{pgfscope}
\end{pgfscope}
\end{tikzpicture}}}+\phi)\psi^2}+\psi^3
$$

As the next step, we refine the above decomposition even further. To be more precise, we employ the localization  operators $\UU_{>}
$ and $\UU_{\leqslant}$ such that $\UU_{>}
+\UU_{\leqslant} = \mathrm{Id}$  (see Section \ref{ssec:local} for their construction) and carefully separate certain contributions of the orange terms above. \rmbb{We point out that the localizers depend on a parameter $L>0$ whose precise value will be determined below in Section \ref{ssec:phi1}. Moreover, we will choose different values of $L$ for different stochastic objects while keeping in mind that $\UU_{>}$ and $\UU_{\leq}$ of one object are given by the same parameter $L$ in order to maintain $\UU_{>}
+\UU_{\leqslant} = \mathrm{Id}$.}

\rmbb{Following the  regularity rules outlined at the beginning of Section \ref{s:phi}}, all the orange terms will be written as a sum of  magenta and blue terms, which will lead to our final decomposition.  Namely,
\begin{align*}
&\rmg{3\llbracket X^2 \rrbracket\succ(\phi+\psi)}=\rmm{3\UU_>\llbracket X^2 \rrbracket\succ(\phi+\psi)}+\rmb{3\UU_\leq\llbracket X^2 \rrbracket\succ(\phi+\psi)}
\end{align*}
\begin{align*}
\rmg{3\llbracket X^2 \rrbracket\prec(\phi+\psi)}&=\rmm{3\UU_>\llbracket X^2 \rrbracket\prec(\phi+\psi)}+\rmb{3\UU_\leqslant\llbracket X^2 \rrbracket\prec(\phi+\psi)},
\end{align*}
\begin{align*}
-\rmg{3( - X^{\!\resizebox{0.6em}{!}{
\begin{tikzpicture}
\pgfpathmoveto{\pgfqpoint{0cm}{-0.035cm}}
\pgfpathlineto{\pgfqpoint{1.376cm}{-0.035cm}}
\pgfpathlineto{\pgfqpoint{1.376cm}{1.552cm}}
\pgfpathlineto{\pgfqpoint{0cm}{1.552cm}}
\pgfpathclose
\pgfusepath{clip}
\begin{pgfscope}
\begin{pgfscope}
\pgfpathmoveto{\pgfqpoint{0cm}{-0.035cm}}
\pgfpathlineto{\pgfqpoint{1.376cm}{-0.035cm}}
\pgfpathlineto{\pgfqpoint{1.376cm}{1.552cm}}
\pgfpathlineto{\pgfqpoint{0cm}{1.552cm}}
\pgfpathclose
\pgfusepath{clip}
\begin{pgfscope}
\begin{pgfscope}
\pgfsetdash{}{0cm}
\pgfsetlinewidth{0.818mm}
\pgfsetroundcap
\pgfsetroundjoin
\pgfsetmiterlimit{7.0}
\definecolor{eps2pgf_color}{gray}{0}\pgfsetstrokecolor{eps2pgf_color}\pgfsetfillcolor{eps2pgf_color}
\pgfpathmoveto{\pgfqpoint{0.117cm}{1.421cm}}
\pgfpathlineto{\pgfqpoint{0.682cm}{0.671cm}}
\pgfpathlineto{\pgfqpoint{1.246cm}{1.421cm}}
\pgfusepath{stroke}
\end{pgfscope}
\definecolor{eps2pgf_color}{gray}{0}\pgfsetstrokecolor{eps2pgf_color}\pgfsetfillcolor{eps2pgf_color}
\pgfpathmoveto{\pgfqpoint{0.273cm}{1.395cm}}
\pgfpathcurveto{\pgfqpoint{0.273cm}{1.432cm}}{\pgfqpoint{0.259cm}{1.467cm}}{\pgfqpoint{0.233cm}{1.492cm}}
\pgfpathcurveto{\pgfqpoint{0.207cm}{1.518cm}}{\pgfqpoint{0.173cm}{1.532cm}}{\pgfqpoint{0.137cm}{1.532cm}}
\pgfpathcurveto{\pgfqpoint{0.1cm}{1.532cm}}{\pgfqpoint{0.066cm}{1.518cm}}{\pgfqpoint{0.04cm}{1.492cm}}
\pgfpathcurveto{\pgfqpoint{0.014cm}{1.467cm}}{\pgfqpoint{0cm}{1.432cm}}{\pgfqpoint{0cm}{1.395cm}}
\pgfpathcurveto{\pgfqpoint{0cm}{1.359cm}}{\pgfqpoint{0.014cm}{1.324cm}}{\pgfqpoint{0.04cm}{1.299cm}}
\pgfpathcurveto{\pgfqpoint{0.066cm}{1.273cm}}{\pgfqpoint{0.1cm}{1.258cm}}{\pgfqpoint{0.137cm}{1.258cm}}
\pgfpathcurveto{\pgfqpoint{0.173cm}{1.258cm}}{\pgfqpoint{0.207cm}{1.273cm}}{\pgfqpoint{0.233cm}{1.299cm}}
\pgfpathcurveto{\pgfqpoint{0.259cm}{1.324cm}}{\pgfqpoint{0.273cm}{1.359cm}}{\pgfqpoint{0.273cm}{1.395cm}}
\pgfusepath{fill}
\begin{pgfscope}
\pgfsetdash{}{0cm}
\pgfsetlinewidth{0.818mm}
\pgfsetmiterlimit{7.0}
\pgfpathmoveto{\pgfqpoint{0.682cm}{0.671cm}}
\pgfpathlineto{\pgfqpoint{0.679cm}{1.418cm}}
\pgfusepath{stroke}
\end{pgfscope}
\pgfpathmoveto{\pgfqpoint{0.815cm}{1.399cm}}
\pgfpathcurveto{\pgfqpoint{0.815cm}{1.435cm}}{\pgfqpoint{0.801cm}{1.47cm}}{\pgfqpoint{0.775cm}{1.496cm}}
\pgfpathcurveto{\pgfqpoint{0.75cm}{1.521cm}}{\pgfqpoint{0.715cm}{1.536cm}}{\pgfqpoint{0.679cm}{1.536cm}}
\pgfpathcurveto{\pgfqpoint{0.643cm}{1.536cm}}{\pgfqpoint{0.608cm}{1.521cm}}{\pgfqpoint{0.582cm}{1.496cm}}
\pgfpathcurveto{\pgfqpoint{0.557cm}{1.47cm}}{\pgfqpoint{0.542cm}{1.435cm}}{\pgfqpoint{0.542cm}{1.399cm}}
\pgfpathcurveto{\pgfqpoint{0.542cm}{1.363cm}}{\pgfqpoint{0.557cm}{1.328cm}}{\pgfqpoint{0.582cm}{1.302cm}}
\pgfpathcurveto{\pgfqpoint{0.608cm}{1.276cm}}{\pgfqpoint{0.643cm}{1.262cm}}{\pgfqpoint{0.679cm}{1.262cm}}
\pgfpathcurveto{\pgfqpoint{0.715cm}{1.262cm}}{\pgfqpoint{0.75cm}{1.276cm}}{\pgfqpoint{0.775cm}{1.302cm}}
\pgfpathcurveto{\pgfqpoint{0.801cm}{1.328cm}}{\pgfqpoint{0.815cm}{1.363cm}}{\pgfqpoint{0.815cm}{1.399cm}}
\pgfusepath{fill}
\pgfpathmoveto{\pgfqpoint{1.345cm}{1.371cm}}
\pgfpathcurveto{\pgfqpoint{1.345cm}{1.408cm}}{\pgfqpoint{1.331cm}{1.442cm}}{\pgfqpoint{1.305cm}{1.468cm}}
\pgfpathcurveto{\pgfqpoint{1.28cm}{1.494cm}}{\pgfqpoint{1.245cm}{1.508cm}}{\pgfqpoint{1.209cm}{1.508cm}}
\pgfpathcurveto{\pgfqpoint{1.172cm}{1.508cm}}{\pgfqpoint{1.138cm}{1.494cm}}{\pgfqpoint{1.112cm}{1.468cm}}
\pgfpathcurveto{\pgfqpoint{1.087cm}{1.442cm}}{\pgfqpoint{1.072cm}{1.408cm}}{\pgfqpoint{1.072cm}{1.371cm}}
\pgfpathcurveto{\pgfqpoint{1.072cm}{1.335cm}}{\pgfqpoint{1.087cm}{1.3cm}}{\pgfqpoint{1.112cm}{1.274cm}}
\pgfpathcurveto{\pgfqpoint{1.138cm}{1.249cm}}{\pgfqpoint{1.172cm}{1.234cm}}{\pgfqpoint{1.209cm}{1.234cm}}
\pgfpathcurveto{\pgfqpoint{1.245cm}{1.234cm}}{\pgfqpoint{1.28cm}{1.249cm}}{\pgfqpoint{1.305cm}{1.274cm}}
\pgfpathcurveto{\pgfqpoint{1.331cm}{1.3cm}}{\pgfqpoint{1.345cm}{1.335cm}}{\pgfqpoint{1.345cm}{1.371cm}}
\pgfusepath{fill}
\begin{pgfscope}
\pgfsetdash{}{0cm}
\pgfsetlinewidth{0.818mm}
\pgfsetroundcap
\pgfsetmiterlimit{4.0}
\pgfpathmoveto{\pgfqpoint{0.682cm}{0.671cm}}
\pgfpathlineto{\pgfqpoint{0.682cm}{0.042cm}}
\pgfusepath{stroke}
\end{pgfscope}
\end{pgfscope}
\end{pgfscope}
\end{pgfscope}
\end{tikzpicture}}} + \phi + \psi)X^{\!\resizebox{!}{.8em}{
\begin{tikzpicture}
\pgfpathmoveto{\pgfqpoint{0cm}{-0.035cm}}
\pgfpathlineto{\pgfqpoint{1.976cm}{-0.035cm}}
\pgfpathlineto{\pgfqpoint{1.976cm}{1.94cm}}
\pgfpathlineto{\pgfqpoint{0cm}{1.94cm}}
\pgfpathclose
\pgfusepath{clip}
\begin{pgfscope}
\begin{pgfscope}
\pgfpathmoveto{\pgfqpoint{0cm}{-0.035cm}}
\pgfpathlineto{\pgfqpoint{1.976cm}{-0.035cm}}
\pgfpathlineto{\pgfqpoint{1.976cm}{1.94cm}}
\pgfpathlineto{\pgfqpoint{0cm}{1.94cm}}
\pgfpathclose
\pgfusepath{clip}
\begin{pgfscope}
\begin{pgfscope}
\pgfsetdash{}{0cm}
\pgfsetlinewidth{0.818mm}
\pgfsetroundcap
\pgfsetroundjoin
\pgfsetmiterlimit{7.0}
\definecolor{eps2pgf_color}{gray}{0}\pgfsetstrokecolor{eps2pgf_color}\pgfsetfillcolor{eps2pgf_color}
\pgfpathmoveto{\pgfqpoint{0.117cm}{1.815cm}}
\pgfpathlineto{\pgfqpoint{0.682cm}{1.065cm}}
\pgfpathlineto{\pgfqpoint{1.246cm}{1.815cm}}
\pgfusepath{stroke}
\end{pgfscope}
\definecolor{eps2pgf_color}{gray}{0}\pgfsetstrokecolor{eps2pgf_color}\pgfsetfillcolor{eps2pgf_color}
\pgfpathmoveto{\pgfqpoint{0.273cm}{1.789cm}}
\pgfpathcurveto{\pgfqpoint{0.273cm}{1.825cm}}{\pgfqpoint{0.259cm}{1.86cm}}{\pgfqpoint{0.233cm}{1.886cm}}
\pgfpathcurveto{\pgfqpoint{0.207cm}{1.912cm}}{\pgfqpoint{0.173cm}{1.926cm}}{\pgfqpoint{0.137cm}{1.926cm}}
\pgfpathcurveto{\pgfqpoint{0.1cm}{1.926cm}}{\pgfqpoint{0.066cm}{1.912cm}}{\pgfqpoint{0.04cm}{1.886cm}}
\pgfpathcurveto{\pgfqpoint{0.014cm}{1.86cm}}{\pgfqpoint{0cm}{1.825cm}}{\pgfqpoint{0cm}{1.789cm}}
\pgfpathcurveto{\pgfqpoint{0cm}{1.753cm}}{\pgfqpoint{0.014cm}{1.718cm}}{\pgfqpoint{0.04cm}{1.692cm}}
\pgfpathcurveto{\pgfqpoint{0.066cm}{1.667cm}}{\pgfqpoint{0.1cm}{1.652cm}}{\pgfqpoint{0.137cm}{1.652cm}}
\pgfpathcurveto{\pgfqpoint{0.173cm}{1.652cm}}{\pgfqpoint{0.207cm}{1.667cm}}{\pgfqpoint{0.233cm}{1.692cm}}
\pgfpathcurveto{\pgfqpoint{0.259cm}{1.718cm}}{\pgfqpoint{0.273cm}{1.753cm}}{\pgfqpoint{0.273cm}{1.789cm}}
\pgfusepath{fill}
\pgfpathmoveto{\pgfqpoint{1.345cm}{1.765cm}}
\pgfpathcurveto{\pgfqpoint{1.345cm}{1.801cm}}{\pgfqpoint{1.331cm}{1.836cm}}{\pgfqpoint{1.305cm}{1.862cm}}
\pgfpathcurveto{\pgfqpoint{1.28cm}{1.887cm}}{\pgfqpoint{1.245cm}{1.902cm}}{\pgfqpoint{1.209cm}{1.902cm}}
\pgfpathcurveto{\pgfqpoint{1.172cm}{1.902cm}}{\pgfqpoint{1.138cm}{1.887cm}}{\pgfqpoint{1.112cm}{1.862cm}}
\pgfpathcurveto{\pgfqpoint{1.087cm}{1.836cm}}{\pgfqpoint{1.072cm}{1.801cm}}{\pgfqpoint{1.072cm}{1.765cm}}
\pgfpathcurveto{\pgfqpoint{1.072cm}{1.728cm}}{\pgfqpoint{1.087cm}{1.694cm}}{\pgfqpoint{1.112cm}{1.668cm}}
\pgfpathcurveto{\pgfqpoint{1.138cm}{1.642cm}}{\pgfqpoint{1.172cm}{1.628cm}}{\pgfqpoint{1.209cm}{1.628cm}}
\pgfpathcurveto{\pgfqpoint{1.245cm}{1.628cm}}{\pgfqpoint{1.28cm}{1.642cm}}{\pgfqpoint{1.305cm}{1.668cm}}
\pgfpathcurveto{\pgfqpoint{1.331cm}{1.694cm}}{\pgfqpoint{1.345cm}{1.728cm}}{\pgfqpoint{1.345cm}{1.765cm}}
\pgfusepath{fill}
\begin{pgfscope}
\pgfsetdash{}{0cm}
\pgfsetlinewidth{0.818mm}
\pgfsetroundcap
\pgfsetroundjoin
\pgfsetmiterlimit{7.0}
\pgfpathmoveto{\pgfqpoint{0.682cm}{1.065cm}}
\pgfpathlineto{\pgfqpoint{1.246cm}{0.315cm}}
\pgfpathlineto{\pgfqpoint{1.811cm}{1.065cm}}
\pgfusepath{stroke}
\end{pgfscope}
\pgfpathmoveto{\pgfqpoint{1.948cm}{1.065cm}}
\pgfpathcurveto{\pgfqpoint{1.948cm}{1.101cm}}{\pgfqpoint{1.933cm}{1.136cm}}{\pgfqpoint{1.907cm}{1.162cm}}
\pgfpathcurveto{\pgfqpoint{1.882cm}{1.187cm}}{\pgfqpoint{1.847cm}{1.202cm}}{\pgfqpoint{1.811cm}{1.202cm}}
\pgfpathcurveto{\pgfqpoint{1.775cm}{1.202cm}}{\pgfqpoint{1.74cm}{1.187cm}}{\pgfqpoint{1.714cm}{1.162cm}}
\pgfpathcurveto{\pgfqpoint{1.689cm}{1.136cm}}{\pgfqpoint{1.674cm}{1.101cm}}{\pgfqpoint{1.674cm}{1.065cm}}
\pgfpathcurveto{\pgfqpoint{1.674cm}{1.029cm}}{\pgfqpoint{1.689cm}{0.994cm}}{\pgfqpoint{1.714cm}{0.968cm}}
\pgfpathcurveto{\pgfqpoint{1.74cm}{0.942cm}}{\pgfqpoint{1.775cm}{0.928cm}}{\pgfqpoint{1.811cm}{0.928cm}}
\pgfpathcurveto{\pgfqpoint{1.847cm}{0.928cm}}{\pgfqpoint{1.882cm}{0.942cm}}{\pgfqpoint{1.907cm}{0.968cm}}
\pgfpathcurveto{\pgfqpoint{1.933cm}{0.994cm}}{\pgfqpoint{1.948cm}{1.029cm}}{\pgfqpoint{1.948cm}{1.065cm}}
\pgfusepath{fill}
\begin{pgfscope}
\pgfsetdash{}{0cm}
\pgfsetlinewidth{0.818mm}
\pgfsetmiterlimit{7.0}
\pgfpathmoveto{\pgfqpoint{1.246cm}{0.315cm}}
\pgfpathlineto{\pgfqpoint{1.244cm}{1.061cm}}
\pgfusepath{stroke}
\end{pgfscope}
\pgfpathmoveto{\pgfqpoint{1.38cm}{1.065cm}}
\pgfpathcurveto{\pgfqpoint{1.38cm}{1.101cm}}{\pgfqpoint{1.366cm}{1.136cm}}{\pgfqpoint{1.34cm}{1.162cm}}
\pgfpathcurveto{\pgfqpoint{1.315cm}{1.187cm}}{\pgfqpoint{1.28cm}{1.202cm}}{\pgfqpoint{1.244cm}{1.202cm}}
\pgfpathcurveto{\pgfqpoint{1.207cm}{1.202cm}}{\pgfqpoint{1.173cm}{1.187cm}}{\pgfqpoint{1.147cm}{1.162cm}}
\pgfpathcurveto{\pgfqpoint{1.121cm}{1.136cm}}{\pgfqpoint{1.107cm}{1.101cm}}{\pgfqpoint{1.107cm}{1.065cm}}
\pgfpathcurveto{\pgfqpoint{1.107cm}{1.029cm}}{\pgfqpoint{1.121cm}{0.994cm}}{\pgfqpoint{1.147cm}{0.968cm}}
\pgfpathcurveto{\pgfqpoint{1.173cm}{0.942cm}}{\pgfqpoint{1.207cm}{0.928cm}}{\pgfqpoint{1.244cm}{0.928cm}}
\pgfpathcurveto{\pgfqpoint{1.28cm}{0.928cm}}{\pgfqpoint{1.315cm}{0.942cm}}{\pgfqpoint{1.34cm}{0.968cm}}
\pgfpathcurveto{\pgfqpoint{1.366cm}{0.994cm}}{\pgfqpoint{1.38cm}{1.029cm}}{\pgfqpoint{1.38cm}{1.065cm}}
\pgfusepath{fill}
\begin{pgfscope}
\pgfsetdash{}{0cm}
\pgfsetlinewidth{0.818mm}
\pgfsetmiterlimit{4.0}
\pgfpathmoveto{\pgfqpoint{1.383cm}{0.178cm}}
\pgfpathcurveto{\pgfqpoint{1.383cm}{0.214cm}}{\pgfqpoint{1.369cm}{0.249cm}}{\pgfqpoint{1.343cm}{0.275cm}}
\pgfpathcurveto{\pgfqpoint{1.317cm}{0.3cm}}{\pgfqpoint{1.283cm}{0.315cm}}{\pgfqpoint{1.246cm}{0.315cm}}
\pgfpathcurveto{\pgfqpoint{1.21cm}{0.315cm}}{\pgfqpoint{1.175cm}{0.3cm}}{\pgfqpoint{1.15cm}{0.275cm}}
\pgfpathcurveto{\pgfqpoint{1.124cm}{0.249cm}}{\pgfqpoint{1.11cm}{0.214cm}}{\pgfqpoint{1.11cm}{0.178cm}}
\pgfpathcurveto{\pgfqpoint{1.11cm}{0.141cm}}{\pgfqpoint{1.124cm}{0.107cm}}{\pgfqpoint{1.15cm}{0.081cm}}
\pgfpathcurveto{\pgfqpoint{1.175cm}{0.055cm}}{\pgfqpoint{1.21cm}{0.041cm}}{\pgfqpoint{1.246cm}{0.041cm}}
\pgfpathcurveto{\pgfqpoint{1.283cm}{0.041cm}}{\pgfqpoint{1.317cm}{0.055cm}}{\pgfqpoint{1.343cm}{0.081cm}}
\pgfpathcurveto{\pgfqpoint{1.369cm}{0.107cm}}{\pgfqpoint{1.383cm}{0.141cm}}{\pgfqpoint{1.383cm}{0.178cm}}
\pgfusepath{stroke}
\end{pgfscope}
\end{pgfscope}
\end{pgfscope}
\end{pgfscope}
\end{tikzpicture}}}}&=\rmm{3X^{\!\resizebox{0.6em}{!}{
\begin{tikzpicture}
\pgfpathmoveto{\pgfqpoint{0cm}{-0.035cm}}
\pgfpathlineto{\pgfqpoint{1.376cm}{-0.035cm}}
\pgfpathlineto{\pgfqpoint{1.376cm}{1.552cm}}
\pgfpathlineto{\pgfqpoint{0cm}{1.552cm}}
\pgfpathclose
\pgfusepath{clip}
\begin{pgfscope}
\begin{pgfscope}
\pgfpathmoveto{\pgfqpoint{0cm}{-0.035cm}}
\pgfpathlineto{\pgfqpoint{1.376cm}{-0.035cm}}
\pgfpathlineto{\pgfqpoint{1.376cm}{1.552cm}}
\pgfpathlineto{\pgfqpoint{0cm}{1.552cm}}
\pgfpathclose
\pgfusepath{clip}
\begin{pgfscope}
\begin{pgfscope}
\pgfsetdash{}{0cm}
\pgfsetlinewidth{0.818mm}
\pgfsetroundcap
\pgfsetroundjoin
\pgfsetmiterlimit{7.0}
\definecolor{eps2pgf_color}{gray}{0}\pgfsetstrokecolor{eps2pgf_color}\pgfsetfillcolor{eps2pgf_color}
\pgfpathmoveto{\pgfqpoint{0.117cm}{1.421cm}}
\pgfpathlineto{\pgfqpoint{0.682cm}{0.671cm}}
\pgfpathlineto{\pgfqpoint{1.246cm}{1.421cm}}
\pgfusepath{stroke}
\end{pgfscope}
\definecolor{eps2pgf_color}{gray}{0}\pgfsetstrokecolor{eps2pgf_color}\pgfsetfillcolor{eps2pgf_color}
\pgfpathmoveto{\pgfqpoint{0.273cm}{1.395cm}}
\pgfpathcurveto{\pgfqpoint{0.273cm}{1.432cm}}{\pgfqpoint{0.259cm}{1.467cm}}{\pgfqpoint{0.233cm}{1.492cm}}
\pgfpathcurveto{\pgfqpoint{0.207cm}{1.518cm}}{\pgfqpoint{0.173cm}{1.532cm}}{\pgfqpoint{0.137cm}{1.532cm}}
\pgfpathcurveto{\pgfqpoint{0.1cm}{1.532cm}}{\pgfqpoint{0.066cm}{1.518cm}}{\pgfqpoint{0.04cm}{1.492cm}}
\pgfpathcurveto{\pgfqpoint{0.014cm}{1.467cm}}{\pgfqpoint{0cm}{1.432cm}}{\pgfqpoint{0cm}{1.395cm}}
\pgfpathcurveto{\pgfqpoint{0cm}{1.359cm}}{\pgfqpoint{0.014cm}{1.324cm}}{\pgfqpoint{0.04cm}{1.299cm}}
\pgfpathcurveto{\pgfqpoint{0.066cm}{1.273cm}}{\pgfqpoint{0.1cm}{1.258cm}}{\pgfqpoint{0.137cm}{1.258cm}}
\pgfpathcurveto{\pgfqpoint{0.173cm}{1.258cm}}{\pgfqpoint{0.207cm}{1.273cm}}{\pgfqpoint{0.233cm}{1.299cm}}
\pgfpathcurveto{\pgfqpoint{0.259cm}{1.324cm}}{\pgfqpoint{0.273cm}{1.359cm}}{\pgfqpoint{0.273cm}{1.395cm}}
\pgfusepath{fill}
\begin{pgfscope}
\pgfsetdash{}{0cm}
\pgfsetlinewidth{0.818mm}
\pgfsetmiterlimit{7.0}
\pgfpathmoveto{\pgfqpoint{0.682cm}{0.671cm}}
\pgfpathlineto{\pgfqpoint{0.679cm}{1.418cm}}
\pgfusepath{stroke}
\end{pgfscope}
\pgfpathmoveto{\pgfqpoint{0.815cm}{1.399cm}}
\pgfpathcurveto{\pgfqpoint{0.815cm}{1.435cm}}{\pgfqpoint{0.801cm}{1.47cm}}{\pgfqpoint{0.775cm}{1.496cm}}
\pgfpathcurveto{\pgfqpoint{0.75cm}{1.521cm}}{\pgfqpoint{0.715cm}{1.536cm}}{\pgfqpoint{0.679cm}{1.536cm}}
\pgfpathcurveto{\pgfqpoint{0.643cm}{1.536cm}}{\pgfqpoint{0.608cm}{1.521cm}}{\pgfqpoint{0.582cm}{1.496cm}}
\pgfpathcurveto{\pgfqpoint{0.557cm}{1.47cm}}{\pgfqpoint{0.542cm}{1.435cm}}{\pgfqpoint{0.542cm}{1.399cm}}
\pgfpathcurveto{\pgfqpoint{0.542cm}{1.363cm}}{\pgfqpoint{0.557cm}{1.328cm}}{\pgfqpoint{0.582cm}{1.302cm}}
\pgfpathcurveto{\pgfqpoint{0.608cm}{1.276cm}}{\pgfqpoint{0.643cm}{1.262cm}}{\pgfqpoint{0.679cm}{1.262cm}}
\pgfpathcurveto{\pgfqpoint{0.715cm}{1.262cm}}{\pgfqpoint{0.75cm}{1.276cm}}{\pgfqpoint{0.775cm}{1.302cm}}
\pgfpathcurveto{\pgfqpoint{0.801cm}{1.328cm}}{\pgfqpoint{0.815cm}{1.363cm}}{\pgfqpoint{0.815cm}{1.399cm}}
\pgfusepath{fill}
\pgfpathmoveto{\pgfqpoint{1.345cm}{1.371cm}}
\pgfpathcurveto{\pgfqpoint{1.345cm}{1.408cm}}{\pgfqpoint{1.331cm}{1.442cm}}{\pgfqpoint{1.305cm}{1.468cm}}
\pgfpathcurveto{\pgfqpoint{1.28cm}{1.494cm}}{\pgfqpoint{1.245cm}{1.508cm}}{\pgfqpoint{1.209cm}{1.508cm}}
\pgfpathcurveto{\pgfqpoint{1.172cm}{1.508cm}}{\pgfqpoint{1.138cm}{1.494cm}}{\pgfqpoint{1.112cm}{1.468cm}}
\pgfpathcurveto{\pgfqpoint{1.087cm}{1.442cm}}{\pgfqpoint{1.072cm}{1.408cm}}{\pgfqpoint{1.072cm}{1.371cm}}
\pgfpathcurveto{\pgfqpoint{1.072cm}{1.335cm}}{\pgfqpoint{1.087cm}{1.3cm}}{\pgfqpoint{1.112cm}{1.274cm}}
\pgfpathcurveto{\pgfqpoint{1.138cm}{1.249cm}}{\pgfqpoint{1.172cm}{1.234cm}}{\pgfqpoint{1.209cm}{1.234cm}}
\pgfpathcurveto{\pgfqpoint{1.245cm}{1.234cm}}{\pgfqpoint{1.28cm}{1.249cm}}{\pgfqpoint{1.305cm}{1.274cm}}
\pgfpathcurveto{\pgfqpoint{1.331cm}{1.3cm}}{\pgfqpoint{1.345cm}{1.335cm}}{\pgfqpoint{1.345cm}{1.371cm}}
\pgfusepath{fill}
\begin{pgfscope}
\pgfsetdash{}{0cm}
\pgfsetlinewidth{0.818mm}
\pgfsetroundcap
\pgfsetmiterlimit{4.0}
\pgfpathmoveto{\pgfqpoint{0.682cm}{0.671cm}}
\pgfpathlineto{\pgfqpoint{0.682cm}{0.042cm}}
\pgfusepath{stroke}
\end{pgfscope}
\end{pgfscope}
\end{pgfscope}
\end{pgfscope}
\end{tikzpicture}}}X^{\!\resizebox{!}{.8em}{
\begin{tikzpicture}
\pgfpathmoveto{\pgfqpoint{0cm}{-0.035cm}}
\pgfpathlineto{\pgfqpoint{1.976cm}{-0.035cm}}
\pgfpathlineto{\pgfqpoint{1.976cm}{1.94cm}}
\pgfpathlineto{\pgfqpoint{0cm}{1.94cm}}
\pgfpathclose
\pgfusepath{clip}
\begin{pgfscope}
\begin{pgfscope}
\pgfpathmoveto{\pgfqpoint{0cm}{-0.035cm}}
\pgfpathlineto{\pgfqpoint{1.976cm}{-0.035cm}}
\pgfpathlineto{\pgfqpoint{1.976cm}{1.94cm}}
\pgfpathlineto{\pgfqpoint{0cm}{1.94cm}}
\pgfpathclose
\pgfusepath{clip}
\begin{pgfscope}
\begin{pgfscope}
\pgfsetdash{}{0cm}
\pgfsetlinewidth{0.818mm}
\pgfsetroundcap
\pgfsetroundjoin
\pgfsetmiterlimit{7.0}
\definecolor{eps2pgf_color}{gray}{0}\pgfsetstrokecolor{eps2pgf_color}\pgfsetfillcolor{eps2pgf_color}
\pgfpathmoveto{\pgfqpoint{0.117cm}{1.815cm}}
\pgfpathlineto{\pgfqpoint{0.682cm}{1.065cm}}
\pgfpathlineto{\pgfqpoint{1.246cm}{1.815cm}}
\pgfusepath{stroke}
\end{pgfscope}
\definecolor{eps2pgf_color}{gray}{0}\pgfsetstrokecolor{eps2pgf_color}\pgfsetfillcolor{eps2pgf_color}
\pgfpathmoveto{\pgfqpoint{0.273cm}{1.789cm}}
\pgfpathcurveto{\pgfqpoint{0.273cm}{1.825cm}}{\pgfqpoint{0.259cm}{1.86cm}}{\pgfqpoint{0.233cm}{1.886cm}}
\pgfpathcurveto{\pgfqpoint{0.207cm}{1.912cm}}{\pgfqpoint{0.173cm}{1.926cm}}{\pgfqpoint{0.137cm}{1.926cm}}
\pgfpathcurveto{\pgfqpoint{0.1cm}{1.926cm}}{\pgfqpoint{0.066cm}{1.912cm}}{\pgfqpoint{0.04cm}{1.886cm}}
\pgfpathcurveto{\pgfqpoint{0.014cm}{1.86cm}}{\pgfqpoint{0cm}{1.825cm}}{\pgfqpoint{0cm}{1.789cm}}
\pgfpathcurveto{\pgfqpoint{0cm}{1.753cm}}{\pgfqpoint{0.014cm}{1.718cm}}{\pgfqpoint{0.04cm}{1.692cm}}
\pgfpathcurveto{\pgfqpoint{0.066cm}{1.667cm}}{\pgfqpoint{0.1cm}{1.652cm}}{\pgfqpoint{0.137cm}{1.652cm}}
\pgfpathcurveto{\pgfqpoint{0.173cm}{1.652cm}}{\pgfqpoint{0.207cm}{1.667cm}}{\pgfqpoint{0.233cm}{1.692cm}}
\pgfpathcurveto{\pgfqpoint{0.259cm}{1.718cm}}{\pgfqpoint{0.273cm}{1.753cm}}{\pgfqpoint{0.273cm}{1.789cm}}
\pgfusepath{fill}
\pgfpathmoveto{\pgfqpoint{1.345cm}{1.765cm}}
\pgfpathcurveto{\pgfqpoint{1.345cm}{1.801cm}}{\pgfqpoint{1.331cm}{1.836cm}}{\pgfqpoint{1.305cm}{1.862cm}}
\pgfpathcurveto{\pgfqpoint{1.28cm}{1.887cm}}{\pgfqpoint{1.245cm}{1.902cm}}{\pgfqpoint{1.209cm}{1.902cm}}
\pgfpathcurveto{\pgfqpoint{1.172cm}{1.902cm}}{\pgfqpoint{1.138cm}{1.887cm}}{\pgfqpoint{1.112cm}{1.862cm}}
\pgfpathcurveto{\pgfqpoint{1.087cm}{1.836cm}}{\pgfqpoint{1.072cm}{1.801cm}}{\pgfqpoint{1.072cm}{1.765cm}}
\pgfpathcurveto{\pgfqpoint{1.072cm}{1.728cm}}{\pgfqpoint{1.087cm}{1.694cm}}{\pgfqpoint{1.112cm}{1.668cm}}
\pgfpathcurveto{\pgfqpoint{1.138cm}{1.642cm}}{\pgfqpoint{1.172cm}{1.628cm}}{\pgfqpoint{1.209cm}{1.628cm}}
\pgfpathcurveto{\pgfqpoint{1.245cm}{1.628cm}}{\pgfqpoint{1.28cm}{1.642cm}}{\pgfqpoint{1.305cm}{1.668cm}}
\pgfpathcurveto{\pgfqpoint{1.331cm}{1.694cm}}{\pgfqpoint{1.345cm}{1.728cm}}{\pgfqpoint{1.345cm}{1.765cm}}
\pgfusepath{fill}
\begin{pgfscope}
\pgfsetdash{}{0cm}
\pgfsetlinewidth{0.818mm}
\pgfsetroundcap
\pgfsetroundjoin
\pgfsetmiterlimit{7.0}
\pgfpathmoveto{\pgfqpoint{0.682cm}{1.065cm}}
\pgfpathlineto{\pgfqpoint{1.246cm}{0.315cm}}
\pgfpathlineto{\pgfqpoint{1.811cm}{1.065cm}}
\pgfusepath{stroke}
\end{pgfscope}
\pgfpathmoveto{\pgfqpoint{1.948cm}{1.065cm}}
\pgfpathcurveto{\pgfqpoint{1.948cm}{1.101cm}}{\pgfqpoint{1.933cm}{1.136cm}}{\pgfqpoint{1.907cm}{1.162cm}}
\pgfpathcurveto{\pgfqpoint{1.882cm}{1.187cm}}{\pgfqpoint{1.847cm}{1.202cm}}{\pgfqpoint{1.811cm}{1.202cm}}
\pgfpathcurveto{\pgfqpoint{1.775cm}{1.202cm}}{\pgfqpoint{1.74cm}{1.187cm}}{\pgfqpoint{1.714cm}{1.162cm}}
\pgfpathcurveto{\pgfqpoint{1.689cm}{1.136cm}}{\pgfqpoint{1.674cm}{1.101cm}}{\pgfqpoint{1.674cm}{1.065cm}}
\pgfpathcurveto{\pgfqpoint{1.674cm}{1.029cm}}{\pgfqpoint{1.689cm}{0.994cm}}{\pgfqpoint{1.714cm}{0.968cm}}
\pgfpathcurveto{\pgfqpoint{1.74cm}{0.942cm}}{\pgfqpoint{1.775cm}{0.928cm}}{\pgfqpoint{1.811cm}{0.928cm}}
\pgfpathcurveto{\pgfqpoint{1.847cm}{0.928cm}}{\pgfqpoint{1.882cm}{0.942cm}}{\pgfqpoint{1.907cm}{0.968cm}}
\pgfpathcurveto{\pgfqpoint{1.933cm}{0.994cm}}{\pgfqpoint{1.948cm}{1.029cm}}{\pgfqpoint{1.948cm}{1.065cm}}
\pgfusepath{fill}
\begin{pgfscope}
\pgfsetdash{}{0cm}
\pgfsetlinewidth{0.818mm}
\pgfsetmiterlimit{7.0}
\pgfpathmoveto{\pgfqpoint{1.246cm}{0.315cm}}
\pgfpathlineto{\pgfqpoint{1.244cm}{1.061cm}}
\pgfusepath{stroke}
\end{pgfscope}
\pgfpathmoveto{\pgfqpoint{1.38cm}{1.065cm}}
\pgfpathcurveto{\pgfqpoint{1.38cm}{1.101cm}}{\pgfqpoint{1.366cm}{1.136cm}}{\pgfqpoint{1.34cm}{1.162cm}}
\pgfpathcurveto{\pgfqpoint{1.315cm}{1.187cm}}{\pgfqpoint{1.28cm}{1.202cm}}{\pgfqpoint{1.244cm}{1.202cm}}
\pgfpathcurveto{\pgfqpoint{1.207cm}{1.202cm}}{\pgfqpoint{1.173cm}{1.187cm}}{\pgfqpoint{1.147cm}{1.162cm}}
\pgfpathcurveto{\pgfqpoint{1.121cm}{1.136cm}}{\pgfqpoint{1.107cm}{1.101cm}}{\pgfqpoint{1.107cm}{1.065cm}}
\pgfpathcurveto{\pgfqpoint{1.107cm}{1.029cm}}{\pgfqpoint{1.121cm}{0.994cm}}{\pgfqpoint{1.147cm}{0.968cm}}
\pgfpathcurveto{\pgfqpoint{1.173cm}{0.942cm}}{\pgfqpoint{1.207cm}{0.928cm}}{\pgfqpoint{1.244cm}{0.928cm}}
\pgfpathcurveto{\pgfqpoint{1.28cm}{0.928cm}}{\pgfqpoint{1.315cm}{0.942cm}}{\pgfqpoint{1.34cm}{0.968cm}}
\pgfpathcurveto{\pgfqpoint{1.366cm}{0.994cm}}{\pgfqpoint{1.38cm}{1.029cm}}{\pgfqpoint{1.38cm}{1.065cm}}
\pgfusepath{fill}
\begin{pgfscope}
\pgfsetdash{}{0cm}
\pgfsetlinewidth{0.818mm}
\pgfsetmiterlimit{4.0}
\pgfpathmoveto{\pgfqpoint{1.383cm}{0.178cm}}
\pgfpathcurveto{\pgfqpoint{1.383cm}{0.214cm}}{\pgfqpoint{1.369cm}{0.249cm}}{\pgfqpoint{1.343cm}{0.275cm}}
\pgfpathcurveto{\pgfqpoint{1.317cm}{0.3cm}}{\pgfqpoint{1.283cm}{0.315cm}}{\pgfqpoint{1.246cm}{0.315cm}}
\pgfpathcurveto{\pgfqpoint{1.21cm}{0.315cm}}{\pgfqpoint{1.175cm}{0.3cm}}{\pgfqpoint{1.15cm}{0.275cm}}
\pgfpathcurveto{\pgfqpoint{1.124cm}{0.249cm}}{\pgfqpoint{1.11cm}{0.214cm}}{\pgfqpoint{1.11cm}{0.178cm}}
\pgfpathcurveto{\pgfqpoint{1.11cm}{0.141cm}}{\pgfqpoint{1.124cm}{0.107cm}}{\pgfqpoint{1.15cm}{0.081cm}}
\pgfpathcurveto{\pgfqpoint{1.175cm}{0.055cm}}{\pgfqpoint{1.21cm}{0.041cm}}{\pgfqpoint{1.246cm}{0.041cm}}
\pgfpathcurveto{\pgfqpoint{1.283cm}{0.041cm}}{\pgfqpoint{1.317cm}{0.055cm}}{\pgfqpoint{1.343cm}{0.081cm}}
\pgfpathcurveto{\pgfqpoint{1.369cm}{0.107cm}}{\pgfqpoint{1.383cm}{0.141cm}}{\pgfqpoint{1.383cm}{0.178cm}}
\pgfusepath{stroke}
\end{pgfscope}
\end{pgfscope}
\end{pgfscope}
\end{pgfscope}
\end{tikzpicture}}}}-\rmm{3( \phi + \psi)\prec\UU_>X^{\!\resizebox{!}{.8em}{
\begin{tikzpicture}
\pgfpathmoveto{\pgfqpoint{0cm}{-0.035cm}}
\pgfpathlineto{\pgfqpoint{1.976cm}{-0.035cm}}
\pgfpathlineto{\pgfqpoint{1.976cm}{1.94cm}}
\pgfpathlineto{\pgfqpoint{0cm}{1.94cm}}
\pgfpathclose
\pgfusepath{clip}
\begin{pgfscope}
\begin{pgfscope}
\pgfpathmoveto{\pgfqpoint{0cm}{-0.035cm}}
\pgfpathlineto{\pgfqpoint{1.976cm}{-0.035cm}}
\pgfpathlineto{\pgfqpoint{1.976cm}{1.94cm}}
\pgfpathlineto{\pgfqpoint{0cm}{1.94cm}}
\pgfpathclose
\pgfusepath{clip}
\begin{pgfscope}
\begin{pgfscope}
\pgfsetdash{}{0cm}
\pgfsetlinewidth{0.818mm}
\pgfsetroundcap
\pgfsetroundjoin
\pgfsetmiterlimit{7.0}
\definecolor{eps2pgf_color}{gray}{0}\pgfsetstrokecolor{eps2pgf_color}\pgfsetfillcolor{eps2pgf_color}
\pgfpathmoveto{\pgfqpoint{0.117cm}{1.815cm}}
\pgfpathlineto{\pgfqpoint{0.682cm}{1.065cm}}
\pgfpathlineto{\pgfqpoint{1.246cm}{1.815cm}}
\pgfusepath{stroke}
\end{pgfscope}
\definecolor{eps2pgf_color}{gray}{0}\pgfsetstrokecolor{eps2pgf_color}\pgfsetfillcolor{eps2pgf_color}
\pgfpathmoveto{\pgfqpoint{0.273cm}{1.789cm}}
\pgfpathcurveto{\pgfqpoint{0.273cm}{1.825cm}}{\pgfqpoint{0.259cm}{1.86cm}}{\pgfqpoint{0.233cm}{1.886cm}}
\pgfpathcurveto{\pgfqpoint{0.207cm}{1.912cm}}{\pgfqpoint{0.173cm}{1.926cm}}{\pgfqpoint{0.137cm}{1.926cm}}
\pgfpathcurveto{\pgfqpoint{0.1cm}{1.926cm}}{\pgfqpoint{0.066cm}{1.912cm}}{\pgfqpoint{0.04cm}{1.886cm}}
\pgfpathcurveto{\pgfqpoint{0.014cm}{1.86cm}}{\pgfqpoint{0cm}{1.825cm}}{\pgfqpoint{0cm}{1.789cm}}
\pgfpathcurveto{\pgfqpoint{0cm}{1.753cm}}{\pgfqpoint{0.014cm}{1.718cm}}{\pgfqpoint{0.04cm}{1.692cm}}
\pgfpathcurveto{\pgfqpoint{0.066cm}{1.667cm}}{\pgfqpoint{0.1cm}{1.652cm}}{\pgfqpoint{0.137cm}{1.652cm}}
\pgfpathcurveto{\pgfqpoint{0.173cm}{1.652cm}}{\pgfqpoint{0.207cm}{1.667cm}}{\pgfqpoint{0.233cm}{1.692cm}}
\pgfpathcurveto{\pgfqpoint{0.259cm}{1.718cm}}{\pgfqpoint{0.273cm}{1.753cm}}{\pgfqpoint{0.273cm}{1.789cm}}
\pgfusepath{fill}
\pgfpathmoveto{\pgfqpoint{1.345cm}{1.765cm}}
\pgfpathcurveto{\pgfqpoint{1.345cm}{1.801cm}}{\pgfqpoint{1.331cm}{1.836cm}}{\pgfqpoint{1.305cm}{1.862cm}}
\pgfpathcurveto{\pgfqpoint{1.28cm}{1.887cm}}{\pgfqpoint{1.245cm}{1.902cm}}{\pgfqpoint{1.209cm}{1.902cm}}
\pgfpathcurveto{\pgfqpoint{1.172cm}{1.902cm}}{\pgfqpoint{1.138cm}{1.887cm}}{\pgfqpoint{1.112cm}{1.862cm}}
\pgfpathcurveto{\pgfqpoint{1.087cm}{1.836cm}}{\pgfqpoint{1.072cm}{1.801cm}}{\pgfqpoint{1.072cm}{1.765cm}}
\pgfpathcurveto{\pgfqpoint{1.072cm}{1.728cm}}{\pgfqpoint{1.087cm}{1.694cm}}{\pgfqpoint{1.112cm}{1.668cm}}
\pgfpathcurveto{\pgfqpoint{1.138cm}{1.642cm}}{\pgfqpoint{1.172cm}{1.628cm}}{\pgfqpoint{1.209cm}{1.628cm}}
\pgfpathcurveto{\pgfqpoint{1.245cm}{1.628cm}}{\pgfqpoint{1.28cm}{1.642cm}}{\pgfqpoint{1.305cm}{1.668cm}}
\pgfpathcurveto{\pgfqpoint{1.331cm}{1.694cm}}{\pgfqpoint{1.345cm}{1.728cm}}{\pgfqpoint{1.345cm}{1.765cm}}
\pgfusepath{fill}
\begin{pgfscope}
\pgfsetdash{}{0cm}
\pgfsetlinewidth{0.818mm}
\pgfsetroundcap
\pgfsetroundjoin
\pgfsetmiterlimit{7.0}
\pgfpathmoveto{\pgfqpoint{0.682cm}{1.065cm}}
\pgfpathlineto{\pgfqpoint{1.246cm}{0.315cm}}
\pgfpathlineto{\pgfqpoint{1.811cm}{1.065cm}}
\pgfusepath{stroke}
\end{pgfscope}
\pgfpathmoveto{\pgfqpoint{1.948cm}{1.065cm}}
\pgfpathcurveto{\pgfqpoint{1.948cm}{1.101cm}}{\pgfqpoint{1.933cm}{1.136cm}}{\pgfqpoint{1.907cm}{1.162cm}}
\pgfpathcurveto{\pgfqpoint{1.882cm}{1.187cm}}{\pgfqpoint{1.847cm}{1.202cm}}{\pgfqpoint{1.811cm}{1.202cm}}
\pgfpathcurveto{\pgfqpoint{1.775cm}{1.202cm}}{\pgfqpoint{1.74cm}{1.187cm}}{\pgfqpoint{1.714cm}{1.162cm}}
\pgfpathcurveto{\pgfqpoint{1.689cm}{1.136cm}}{\pgfqpoint{1.674cm}{1.101cm}}{\pgfqpoint{1.674cm}{1.065cm}}
\pgfpathcurveto{\pgfqpoint{1.674cm}{1.029cm}}{\pgfqpoint{1.689cm}{0.994cm}}{\pgfqpoint{1.714cm}{0.968cm}}
\pgfpathcurveto{\pgfqpoint{1.74cm}{0.942cm}}{\pgfqpoint{1.775cm}{0.928cm}}{\pgfqpoint{1.811cm}{0.928cm}}
\pgfpathcurveto{\pgfqpoint{1.847cm}{0.928cm}}{\pgfqpoint{1.882cm}{0.942cm}}{\pgfqpoint{1.907cm}{0.968cm}}
\pgfpathcurveto{\pgfqpoint{1.933cm}{0.994cm}}{\pgfqpoint{1.948cm}{1.029cm}}{\pgfqpoint{1.948cm}{1.065cm}}
\pgfusepath{fill}
\begin{pgfscope}
\pgfsetdash{}{0cm}
\pgfsetlinewidth{0.818mm}
\pgfsetmiterlimit{7.0}
\pgfpathmoveto{\pgfqpoint{1.246cm}{0.315cm}}
\pgfpathlineto{\pgfqpoint{1.244cm}{1.061cm}}
\pgfusepath{stroke}
\end{pgfscope}
\pgfpathmoveto{\pgfqpoint{1.38cm}{1.065cm}}
\pgfpathcurveto{\pgfqpoint{1.38cm}{1.101cm}}{\pgfqpoint{1.366cm}{1.136cm}}{\pgfqpoint{1.34cm}{1.162cm}}
\pgfpathcurveto{\pgfqpoint{1.315cm}{1.187cm}}{\pgfqpoint{1.28cm}{1.202cm}}{\pgfqpoint{1.244cm}{1.202cm}}
\pgfpathcurveto{\pgfqpoint{1.207cm}{1.202cm}}{\pgfqpoint{1.173cm}{1.187cm}}{\pgfqpoint{1.147cm}{1.162cm}}
\pgfpathcurveto{\pgfqpoint{1.121cm}{1.136cm}}{\pgfqpoint{1.107cm}{1.101cm}}{\pgfqpoint{1.107cm}{1.065cm}}
\pgfpathcurveto{\pgfqpoint{1.107cm}{1.029cm}}{\pgfqpoint{1.121cm}{0.994cm}}{\pgfqpoint{1.147cm}{0.968cm}}
\pgfpathcurveto{\pgfqpoint{1.173cm}{0.942cm}}{\pgfqpoint{1.207cm}{0.928cm}}{\pgfqpoint{1.244cm}{0.928cm}}
\pgfpathcurveto{\pgfqpoint{1.28cm}{0.928cm}}{\pgfqpoint{1.315cm}{0.942cm}}{\pgfqpoint{1.34cm}{0.968cm}}
\pgfpathcurveto{\pgfqpoint{1.366cm}{0.994cm}}{\pgfqpoint{1.38cm}{1.029cm}}{\pgfqpoint{1.38cm}{1.065cm}}
\pgfusepath{fill}
\begin{pgfscope}
\pgfsetdash{}{0cm}
\pgfsetlinewidth{0.818mm}
\pgfsetmiterlimit{4.0}
\pgfpathmoveto{\pgfqpoint{1.383cm}{0.178cm}}
\pgfpathcurveto{\pgfqpoint{1.383cm}{0.214cm}}{\pgfqpoint{1.369cm}{0.249cm}}{\pgfqpoint{1.343cm}{0.275cm}}
\pgfpathcurveto{\pgfqpoint{1.317cm}{0.3cm}}{\pgfqpoint{1.283cm}{0.315cm}}{\pgfqpoint{1.246cm}{0.315cm}}
\pgfpathcurveto{\pgfqpoint{1.21cm}{0.315cm}}{\pgfqpoint{1.175cm}{0.3cm}}{\pgfqpoint{1.15cm}{0.275cm}}
\pgfpathcurveto{\pgfqpoint{1.124cm}{0.249cm}}{\pgfqpoint{1.11cm}{0.214cm}}{\pgfqpoint{1.11cm}{0.178cm}}
\pgfpathcurveto{\pgfqpoint{1.11cm}{0.141cm}}{\pgfqpoint{1.124cm}{0.107cm}}{\pgfqpoint{1.15cm}{0.081cm}}
\pgfpathcurveto{\pgfqpoint{1.175cm}{0.055cm}}{\pgfqpoint{1.21cm}{0.041cm}}{\pgfqpoint{1.246cm}{0.041cm}}
\pgfpathcurveto{\pgfqpoint{1.283cm}{0.041cm}}{\pgfqpoint{1.317cm}{0.055cm}}{\pgfqpoint{1.343cm}{0.081cm}}
\pgfpathcurveto{\pgfqpoint{1.369cm}{0.107cm}}{\pgfqpoint{1.383cm}{0.141cm}}{\pgfqpoint{1.383cm}{0.178cm}}
\pgfusepath{stroke}
\end{pgfscope}
\end{pgfscope}
\end{pgfscope}
\end{pgfscope}
\end{tikzpicture}}}}\\
&\quad-\rmb{3( \phi + \psi)\prec\UU_\leqslantX^{\!\resizebox{!}{.8em}{
\begin{tikzpicture}
\pgfpathmoveto{\pgfqpoint{0cm}{-0.035cm}}
\pgfpathlineto{\pgfqpoint{1.976cm}{-0.035cm}}
\pgfpathlineto{\pgfqpoint{1.976cm}{1.94cm}}
\pgfpathlineto{\pgfqpoint{0cm}{1.94cm}}
\pgfpathclose
\pgfusepath{clip}
\begin{pgfscope}
\begin{pgfscope}
\pgfpathmoveto{\pgfqpoint{0cm}{-0.035cm}}
\pgfpathlineto{\pgfqpoint{1.976cm}{-0.035cm}}
\pgfpathlineto{\pgfqpoint{1.976cm}{1.94cm}}
\pgfpathlineto{\pgfqpoint{0cm}{1.94cm}}
\pgfpathclose
\pgfusepath{clip}
\begin{pgfscope}
\begin{pgfscope}
\pgfsetdash{}{0cm}
\pgfsetlinewidth{0.818mm}
\pgfsetroundcap
\pgfsetroundjoin
\pgfsetmiterlimit{7.0}
\definecolor{eps2pgf_color}{gray}{0}\pgfsetstrokecolor{eps2pgf_color}\pgfsetfillcolor{eps2pgf_color}
\pgfpathmoveto{\pgfqpoint{0.117cm}{1.815cm}}
\pgfpathlineto{\pgfqpoint{0.682cm}{1.065cm}}
\pgfpathlineto{\pgfqpoint{1.246cm}{1.815cm}}
\pgfusepath{stroke}
\end{pgfscope}
\definecolor{eps2pgf_color}{gray}{0}\pgfsetstrokecolor{eps2pgf_color}\pgfsetfillcolor{eps2pgf_color}
\pgfpathmoveto{\pgfqpoint{0.273cm}{1.789cm}}
\pgfpathcurveto{\pgfqpoint{0.273cm}{1.825cm}}{\pgfqpoint{0.259cm}{1.86cm}}{\pgfqpoint{0.233cm}{1.886cm}}
\pgfpathcurveto{\pgfqpoint{0.207cm}{1.912cm}}{\pgfqpoint{0.173cm}{1.926cm}}{\pgfqpoint{0.137cm}{1.926cm}}
\pgfpathcurveto{\pgfqpoint{0.1cm}{1.926cm}}{\pgfqpoint{0.066cm}{1.912cm}}{\pgfqpoint{0.04cm}{1.886cm}}
\pgfpathcurveto{\pgfqpoint{0.014cm}{1.86cm}}{\pgfqpoint{0cm}{1.825cm}}{\pgfqpoint{0cm}{1.789cm}}
\pgfpathcurveto{\pgfqpoint{0cm}{1.753cm}}{\pgfqpoint{0.014cm}{1.718cm}}{\pgfqpoint{0.04cm}{1.692cm}}
\pgfpathcurveto{\pgfqpoint{0.066cm}{1.667cm}}{\pgfqpoint{0.1cm}{1.652cm}}{\pgfqpoint{0.137cm}{1.652cm}}
\pgfpathcurveto{\pgfqpoint{0.173cm}{1.652cm}}{\pgfqpoint{0.207cm}{1.667cm}}{\pgfqpoint{0.233cm}{1.692cm}}
\pgfpathcurveto{\pgfqpoint{0.259cm}{1.718cm}}{\pgfqpoint{0.273cm}{1.753cm}}{\pgfqpoint{0.273cm}{1.789cm}}
\pgfusepath{fill}
\pgfpathmoveto{\pgfqpoint{1.345cm}{1.765cm}}
\pgfpathcurveto{\pgfqpoint{1.345cm}{1.801cm}}{\pgfqpoint{1.331cm}{1.836cm}}{\pgfqpoint{1.305cm}{1.862cm}}
\pgfpathcurveto{\pgfqpoint{1.28cm}{1.887cm}}{\pgfqpoint{1.245cm}{1.902cm}}{\pgfqpoint{1.209cm}{1.902cm}}
\pgfpathcurveto{\pgfqpoint{1.172cm}{1.902cm}}{\pgfqpoint{1.138cm}{1.887cm}}{\pgfqpoint{1.112cm}{1.862cm}}
\pgfpathcurveto{\pgfqpoint{1.087cm}{1.836cm}}{\pgfqpoint{1.072cm}{1.801cm}}{\pgfqpoint{1.072cm}{1.765cm}}
\pgfpathcurveto{\pgfqpoint{1.072cm}{1.728cm}}{\pgfqpoint{1.087cm}{1.694cm}}{\pgfqpoint{1.112cm}{1.668cm}}
\pgfpathcurveto{\pgfqpoint{1.138cm}{1.642cm}}{\pgfqpoint{1.172cm}{1.628cm}}{\pgfqpoint{1.209cm}{1.628cm}}
\pgfpathcurveto{\pgfqpoint{1.245cm}{1.628cm}}{\pgfqpoint{1.28cm}{1.642cm}}{\pgfqpoint{1.305cm}{1.668cm}}
\pgfpathcurveto{\pgfqpoint{1.331cm}{1.694cm}}{\pgfqpoint{1.345cm}{1.728cm}}{\pgfqpoint{1.345cm}{1.765cm}}
\pgfusepath{fill}
\begin{pgfscope}
\pgfsetdash{}{0cm}
\pgfsetlinewidth{0.818mm}
\pgfsetroundcap
\pgfsetroundjoin
\pgfsetmiterlimit{7.0}
\pgfpathmoveto{\pgfqpoint{0.682cm}{1.065cm}}
\pgfpathlineto{\pgfqpoint{1.246cm}{0.315cm}}
\pgfpathlineto{\pgfqpoint{1.811cm}{1.065cm}}
\pgfusepath{stroke}
\end{pgfscope}
\pgfpathmoveto{\pgfqpoint{1.948cm}{1.065cm}}
\pgfpathcurveto{\pgfqpoint{1.948cm}{1.101cm}}{\pgfqpoint{1.933cm}{1.136cm}}{\pgfqpoint{1.907cm}{1.162cm}}
\pgfpathcurveto{\pgfqpoint{1.882cm}{1.187cm}}{\pgfqpoint{1.847cm}{1.202cm}}{\pgfqpoint{1.811cm}{1.202cm}}
\pgfpathcurveto{\pgfqpoint{1.775cm}{1.202cm}}{\pgfqpoint{1.74cm}{1.187cm}}{\pgfqpoint{1.714cm}{1.162cm}}
\pgfpathcurveto{\pgfqpoint{1.689cm}{1.136cm}}{\pgfqpoint{1.674cm}{1.101cm}}{\pgfqpoint{1.674cm}{1.065cm}}
\pgfpathcurveto{\pgfqpoint{1.674cm}{1.029cm}}{\pgfqpoint{1.689cm}{0.994cm}}{\pgfqpoint{1.714cm}{0.968cm}}
\pgfpathcurveto{\pgfqpoint{1.74cm}{0.942cm}}{\pgfqpoint{1.775cm}{0.928cm}}{\pgfqpoint{1.811cm}{0.928cm}}
\pgfpathcurveto{\pgfqpoint{1.847cm}{0.928cm}}{\pgfqpoint{1.882cm}{0.942cm}}{\pgfqpoint{1.907cm}{0.968cm}}
\pgfpathcurveto{\pgfqpoint{1.933cm}{0.994cm}}{\pgfqpoint{1.948cm}{1.029cm}}{\pgfqpoint{1.948cm}{1.065cm}}
\pgfusepath{fill}
\begin{pgfscope}
\pgfsetdash{}{0cm}
\pgfsetlinewidth{0.818mm}
\pgfsetmiterlimit{7.0}
\pgfpathmoveto{\pgfqpoint{1.246cm}{0.315cm}}
\pgfpathlineto{\pgfqpoint{1.244cm}{1.061cm}}
\pgfusepath{stroke}
\end{pgfscope}
\pgfpathmoveto{\pgfqpoint{1.38cm}{1.065cm}}
\pgfpathcurveto{\pgfqpoint{1.38cm}{1.101cm}}{\pgfqpoint{1.366cm}{1.136cm}}{\pgfqpoint{1.34cm}{1.162cm}}
\pgfpathcurveto{\pgfqpoint{1.315cm}{1.187cm}}{\pgfqpoint{1.28cm}{1.202cm}}{\pgfqpoint{1.244cm}{1.202cm}}
\pgfpathcurveto{\pgfqpoint{1.207cm}{1.202cm}}{\pgfqpoint{1.173cm}{1.187cm}}{\pgfqpoint{1.147cm}{1.162cm}}
\pgfpathcurveto{\pgfqpoint{1.121cm}{1.136cm}}{\pgfqpoint{1.107cm}{1.101cm}}{\pgfqpoint{1.107cm}{1.065cm}}
\pgfpathcurveto{\pgfqpoint{1.107cm}{1.029cm}}{\pgfqpoint{1.121cm}{0.994cm}}{\pgfqpoint{1.147cm}{0.968cm}}
\pgfpathcurveto{\pgfqpoint{1.173cm}{0.942cm}}{\pgfqpoint{1.207cm}{0.928cm}}{\pgfqpoint{1.244cm}{0.928cm}}
\pgfpathcurveto{\pgfqpoint{1.28cm}{0.928cm}}{\pgfqpoint{1.315cm}{0.942cm}}{\pgfqpoint{1.34cm}{0.968cm}}
\pgfpathcurveto{\pgfqpoint{1.366cm}{0.994cm}}{\pgfqpoint{1.38cm}{1.029cm}}{\pgfqpoint{1.38cm}{1.065cm}}
\pgfusepath{fill}
\begin{pgfscope}
\pgfsetdash{}{0cm}
\pgfsetlinewidth{0.818mm}
\pgfsetmiterlimit{4.0}
\pgfpathmoveto{\pgfqpoint{1.383cm}{0.178cm}}
\pgfpathcurveto{\pgfqpoint{1.383cm}{0.214cm}}{\pgfqpoint{1.369cm}{0.249cm}}{\pgfqpoint{1.343cm}{0.275cm}}
\pgfpathcurveto{\pgfqpoint{1.317cm}{0.3cm}}{\pgfqpoint{1.283cm}{0.315cm}}{\pgfqpoint{1.246cm}{0.315cm}}
\pgfpathcurveto{\pgfqpoint{1.21cm}{0.315cm}}{\pgfqpoint{1.175cm}{0.3cm}}{\pgfqpoint{1.15cm}{0.275cm}}
\pgfpathcurveto{\pgfqpoint{1.124cm}{0.249cm}}{\pgfqpoint{1.11cm}{0.214cm}}{\pgfqpoint{1.11cm}{0.178cm}}
\pgfpathcurveto{\pgfqpoint{1.11cm}{0.141cm}}{\pgfqpoint{1.124cm}{0.107cm}}{\pgfqpoint{1.15cm}{0.081cm}}
\pgfpathcurveto{\pgfqpoint{1.175cm}{0.055cm}}{\pgfqpoint{1.21cm}{0.041cm}}{\pgfqpoint{1.246cm}{0.041cm}}
\pgfpathcurveto{\pgfqpoint{1.283cm}{0.041cm}}{\pgfqpoint{1.317cm}{0.055cm}}{\pgfqpoint{1.343cm}{0.081cm}}
\pgfpathcurveto{\pgfqpoint{1.369cm}{0.107cm}}{\pgfqpoint{1.383cm}{0.141cm}}{\pgfqpoint{1.383cm}{0.178cm}}
\pgfusepath{stroke}
\end{pgfscope}
\end{pgfscope}
\end{pgfscope}
\end{pgfscope}
\end{tikzpicture}}}}-\rmb{3( \phi + \psi)\succcurlyeqX^{\!\resizebox{!}{.8em}{
\begin{tikzpicture}
\pgfpathmoveto{\pgfqpoint{0cm}{-0.035cm}}
\pgfpathlineto{\pgfqpoint{1.976cm}{-0.035cm}}
\pgfpathlineto{\pgfqpoint{1.976cm}{1.94cm}}
\pgfpathlineto{\pgfqpoint{0cm}{1.94cm}}
\pgfpathclose
\pgfusepath{clip}
\begin{pgfscope}
\begin{pgfscope}
\pgfpathmoveto{\pgfqpoint{0cm}{-0.035cm}}
\pgfpathlineto{\pgfqpoint{1.976cm}{-0.035cm}}
\pgfpathlineto{\pgfqpoint{1.976cm}{1.94cm}}
\pgfpathlineto{\pgfqpoint{0cm}{1.94cm}}
\pgfpathclose
\pgfusepath{clip}
\begin{pgfscope}
\begin{pgfscope}
\pgfsetdash{}{0cm}
\pgfsetlinewidth{0.818mm}
\pgfsetroundcap
\pgfsetroundjoin
\pgfsetmiterlimit{7.0}
\definecolor{eps2pgf_color}{gray}{0}\pgfsetstrokecolor{eps2pgf_color}\pgfsetfillcolor{eps2pgf_color}
\pgfpathmoveto{\pgfqpoint{0.117cm}{1.815cm}}
\pgfpathlineto{\pgfqpoint{0.682cm}{1.065cm}}
\pgfpathlineto{\pgfqpoint{1.246cm}{1.815cm}}
\pgfusepath{stroke}
\end{pgfscope}
\definecolor{eps2pgf_color}{gray}{0}\pgfsetstrokecolor{eps2pgf_color}\pgfsetfillcolor{eps2pgf_color}
\pgfpathmoveto{\pgfqpoint{0.273cm}{1.789cm}}
\pgfpathcurveto{\pgfqpoint{0.273cm}{1.825cm}}{\pgfqpoint{0.259cm}{1.86cm}}{\pgfqpoint{0.233cm}{1.886cm}}
\pgfpathcurveto{\pgfqpoint{0.207cm}{1.912cm}}{\pgfqpoint{0.173cm}{1.926cm}}{\pgfqpoint{0.137cm}{1.926cm}}
\pgfpathcurveto{\pgfqpoint{0.1cm}{1.926cm}}{\pgfqpoint{0.066cm}{1.912cm}}{\pgfqpoint{0.04cm}{1.886cm}}
\pgfpathcurveto{\pgfqpoint{0.014cm}{1.86cm}}{\pgfqpoint{0cm}{1.825cm}}{\pgfqpoint{0cm}{1.789cm}}
\pgfpathcurveto{\pgfqpoint{0cm}{1.753cm}}{\pgfqpoint{0.014cm}{1.718cm}}{\pgfqpoint{0.04cm}{1.692cm}}
\pgfpathcurveto{\pgfqpoint{0.066cm}{1.667cm}}{\pgfqpoint{0.1cm}{1.652cm}}{\pgfqpoint{0.137cm}{1.652cm}}
\pgfpathcurveto{\pgfqpoint{0.173cm}{1.652cm}}{\pgfqpoint{0.207cm}{1.667cm}}{\pgfqpoint{0.233cm}{1.692cm}}
\pgfpathcurveto{\pgfqpoint{0.259cm}{1.718cm}}{\pgfqpoint{0.273cm}{1.753cm}}{\pgfqpoint{0.273cm}{1.789cm}}
\pgfusepath{fill}
\pgfpathmoveto{\pgfqpoint{1.345cm}{1.765cm}}
\pgfpathcurveto{\pgfqpoint{1.345cm}{1.801cm}}{\pgfqpoint{1.331cm}{1.836cm}}{\pgfqpoint{1.305cm}{1.862cm}}
\pgfpathcurveto{\pgfqpoint{1.28cm}{1.887cm}}{\pgfqpoint{1.245cm}{1.902cm}}{\pgfqpoint{1.209cm}{1.902cm}}
\pgfpathcurveto{\pgfqpoint{1.172cm}{1.902cm}}{\pgfqpoint{1.138cm}{1.887cm}}{\pgfqpoint{1.112cm}{1.862cm}}
\pgfpathcurveto{\pgfqpoint{1.087cm}{1.836cm}}{\pgfqpoint{1.072cm}{1.801cm}}{\pgfqpoint{1.072cm}{1.765cm}}
\pgfpathcurveto{\pgfqpoint{1.072cm}{1.728cm}}{\pgfqpoint{1.087cm}{1.694cm}}{\pgfqpoint{1.112cm}{1.668cm}}
\pgfpathcurveto{\pgfqpoint{1.138cm}{1.642cm}}{\pgfqpoint{1.172cm}{1.628cm}}{\pgfqpoint{1.209cm}{1.628cm}}
\pgfpathcurveto{\pgfqpoint{1.245cm}{1.628cm}}{\pgfqpoint{1.28cm}{1.642cm}}{\pgfqpoint{1.305cm}{1.668cm}}
\pgfpathcurveto{\pgfqpoint{1.331cm}{1.694cm}}{\pgfqpoint{1.345cm}{1.728cm}}{\pgfqpoint{1.345cm}{1.765cm}}
\pgfusepath{fill}
\begin{pgfscope}
\pgfsetdash{}{0cm}
\pgfsetlinewidth{0.818mm}
\pgfsetroundcap
\pgfsetroundjoin
\pgfsetmiterlimit{7.0}
\pgfpathmoveto{\pgfqpoint{0.682cm}{1.065cm}}
\pgfpathlineto{\pgfqpoint{1.246cm}{0.315cm}}
\pgfpathlineto{\pgfqpoint{1.811cm}{1.065cm}}
\pgfusepath{stroke}
\end{pgfscope}
\pgfpathmoveto{\pgfqpoint{1.948cm}{1.065cm}}
\pgfpathcurveto{\pgfqpoint{1.948cm}{1.101cm}}{\pgfqpoint{1.933cm}{1.136cm}}{\pgfqpoint{1.907cm}{1.162cm}}
\pgfpathcurveto{\pgfqpoint{1.882cm}{1.187cm}}{\pgfqpoint{1.847cm}{1.202cm}}{\pgfqpoint{1.811cm}{1.202cm}}
\pgfpathcurveto{\pgfqpoint{1.775cm}{1.202cm}}{\pgfqpoint{1.74cm}{1.187cm}}{\pgfqpoint{1.714cm}{1.162cm}}
\pgfpathcurveto{\pgfqpoint{1.689cm}{1.136cm}}{\pgfqpoint{1.674cm}{1.101cm}}{\pgfqpoint{1.674cm}{1.065cm}}
\pgfpathcurveto{\pgfqpoint{1.674cm}{1.029cm}}{\pgfqpoint{1.689cm}{0.994cm}}{\pgfqpoint{1.714cm}{0.968cm}}
\pgfpathcurveto{\pgfqpoint{1.74cm}{0.942cm}}{\pgfqpoint{1.775cm}{0.928cm}}{\pgfqpoint{1.811cm}{0.928cm}}
\pgfpathcurveto{\pgfqpoint{1.847cm}{0.928cm}}{\pgfqpoint{1.882cm}{0.942cm}}{\pgfqpoint{1.907cm}{0.968cm}}
\pgfpathcurveto{\pgfqpoint{1.933cm}{0.994cm}}{\pgfqpoint{1.948cm}{1.029cm}}{\pgfqpoint{1.948cm}{1.065cm}}
\pgfusepath{fill}
\begin{pgfscope}
\pgfsetdash{}{0cm}
\pgfsetlinewidth{0.818mm}
\pgfsetmiterlimit{7.0}
\pgfpathmoveto{\pgfqpoint{1.246cm}{0.315cm}}
\pgfpathlineto{\pgfqpoint{1.244cm}{1.061cm}}
\pgfusepath{stroke}
\end{pgfscope}
\pgfpathmoveto{\pgfqpoint{1.38cm}{1.065cm}}
\pgfpathcurveto{\pgfqpoint{1.38cm}{1.101cm}}{\pgfqpoint{1.366cm}{1.136cm}}{\pgfqpoint{1.34cm}{1.162cm}}
\pgfpathcurveto{\pgfqpoint{1.315cm}{1.187cm}}{\pgfqpoint{1.28cm}{1.202cm}}{\pgfqpoint{1.244cm}{1.202cm}}
\pgfpathcurveto{\pgfqpoint{1.207cm}{1.202cm}}{\pgfqpoint{1.173cm}{1.187cm}}{\pgfqpoint{1.147cm}{1.162cm}}
\pgfpathcurveto{\pgfqpoint{1.121cm}{1.136cm}}{\pgfqpoint{1.107cm}{1.101cm}}{\pgfqpoint{1.107cm}{1.065cm}}
\pgfpathcurveto{\pgfqpoint{1.107cm}{1.029cm}}{\pgfqpoint{1.121cm}{0.994cm}}{\pgfqpoint{1.147cm}{0.968cm}}
\pgfpathcurveto{\pgfqpoint{1.173cm}{0.942cm}}{\pgfqpoint{1.207cm}{0.928cm}}{\pgfqpoint{1.244cm}{0.928cm}}
\pgfpathcurveto{\pgfqpoint{1.28cm}{0.928cm}}{\pgfqpoint{1.315cm}{0.942cm}}{\pgfqpoint{1.34cm}{0.968cm}}
\pgfpathcurveto{\pgfqpoint{1.366cm}{0.994cm}}{\pgfqpoint{1.38cm}{1.029cm}}{\pgfqpoint{1.38cm}{1.065cm}}
\pgfusepath{fill}
\begin{pgfscope}
\pgfsetdash{}{0cm}
\pgfsetlinewidth{0.818mm}
\pgfsetmiterlimit{4.0}
\pgfpathmoveto{\pgfqpoint{1.383cm}{0.178cm}}
\pgfpathcurveto{\pgfqpoint{1.383cm}{0.214cm}}{\pgfqpoint{1.369cm}{0.249cm}}{\pgfqpoint{1.343cm}{0.275cm}}
\pgfpathcurveto{\pgfqpoint{1.317cm}{0.3cm}}{\pgfqpoint{1.283cm}{0.315cm}}{\pgfqpoint{1.246cm}{0.315cm}}
\pgfpathcurveto{\pgfqpoint{1.21cm}{0.315cm}}{\pgfqpoint{1.175cm}{0.3cm}}{\pgfqpoint{1.15cm}{0.275cm}}
\pgfpathcurveto{\pgfqpoint{1.124cm}{0.249cm}}{\pgfqpoint{1.11cm}{0.214cm}}{\pgfqpoint{1.11cm}{0.178cm}}
\pgfpathcurveto{\pgfqpoint{1.11cm}{0.141cm}}{\pgfqpoint{1.124cm}{0.107cm}}{\pgfqpoint{1.15cm}{0.081cm}}
\pgfpathcurveto{\pgfqpoint{1.175cm}{0.055cm}}{\pgfqpoint{1.21cm}{0.041cm}}{\pgfqpoint{1.246cm}{0.041cm}}
\pgfpathcurveto{\pgfqpoint{1.283cm}{0.041cm}}{\pgfqpoint{1.317cm}{0.055cm}}{\pgfqpoint{1.343cm}{0.081cm}}
\pgfpathcurveto{\pgfqpoint{1.369cm}{0.107cm}}{\pgfqpoint{1.383cm}{0.141cm}}{\pgfqpoint{1.383cm}{0.178cm}}
\pgfusepath{stroke}
\end{pgfscope}
\end{pgfscope}
\end{pgfscope}
\end{pgfscope}
\end{tikzpicture}}}},
\end{align*}
\begin{align}\label{eq:92}
\rmg{6X\succ(X^{\!\resizebox{0.6em}{!}{
\begin{tikzpicture}
\pgfpathmoveto{\pgfqpoint{0cm}{-0.035cm}}
\pgfpathlineto{\pgfqpoint{1.376cm}{-0.035cm}}
\pgfpathlineto{\pgfqpoint{1.376cm}{1.552cm}}
\pgfpathlineto{\pgfqpoint{0cm}{1.552cm}}
\pgfpathclose
\pgfusepath{clip}
\begin{pgfscope}
\begin{pgfscope}
\pgfpathmoveto{\pgfqpoint{0cm}{-0.035cm}}
\pgfpathlineto{\pgfqpoint{1.376cm}{-0.035cm}}
\pgfpathlineto{\pgfqpoint{1.376cm}{1.552cm}}
\pgfpathlineto{\pgfqpoint{0cm}{1.552cm}}
\pgfpathclose
\pgfusepath{clip}
\begin{pgfscope}
\begin{pgfscope}
\pgfsetdash{}{0cm}
\pgfsetlinewidth{0.818mm}
\pgfsetroundcap
\pgfsetroundjoin
\pgfsetmiterlimit{7.0}
\definecolor{eps2pgf_color}{gray}{0}\pgfsetstrokecolor{eps2pgf_color}\pgfsetfillcolor{eps2pgf_color}
\pgfpathmoveto{\pgfqpoint{0.117cm}{1.421cm}}
\pgfpathlineto{\pgfqpoint{0.682cm}{0.671cm}}
\pgfpathlineto{\pgfqpoint{1.246cm}{1.421cm}}
\pgfusepath{stroke}
\end{pgfscope}
\definecolor{eps2pgf_color}{gray}{0}\pgfsetstrokecolor{eps2pgf_color}\pgfsetfillcolor{eps2pgf_color}
\pgfpathmoveto{\pgfqpoint{0.273cm}{1.395cm}}
\pgfpathcurveto{\pgfqpoint{0.273cm}{1.432cm}}{\pgfqpoint{0.259cm}{1.467cm}}{\pgfqpoint{0.233cm}{1.492cm}}
\pgfpathcurveto{\pgfqpoint{0.207cm}{1.518cm}}{\pgfqpoint{0.173cm}{1.532cm}}{\pgfqpoint{0.137cm}{1.532cm}}
\pgfpathcurveto{\pgfqpoint{0.1cm}{1.532cm}}{\pgfqpoint{0.066cm}{1.518cm}}{\pgfqpoint{0.04cm}{1.492cm}}
\pgfpathcurveto{\pgfqpoint{0.014cm}{1.467cm}}{\pgfqpoint{0cm}{1.432cm}}{\pgfqpoint{0cm}{1.395cm}}
\pgfpathcurveto{\pgfqpoint{0cm}{1.359cm}}{\pgfqpoint{0.014cm}{1.324cm}}{\pgfqpoint{0.04cm}{1.299cm}}
\pgfpathcurveto{\pgfqpoint{0.066cm}{1.273cm}}{\pgfqpoint{0.1cm}{1.258cm}}{\pgfqpoint{0.137cm}{1.258cm}}
\pgfpathcurveto{\pgfqpoint{0.173cm}{1.258cm}}{\pgfqpoint{0.207cm}{1.273cm}}{\pgfqpoint{0.233cm}{1.299cm}}
\pgfpathcurveto{\pgfqpoint{0.259cm}{1.324cm}}{\pgfqpoint{0.273cm}{1.359cm}}{\pgfqpoint{0.273cm}{1.395cm}}
\pgfusepath{fill}
\begin{pgfscope}
\pgfsetdash{}{0cm}
\pgfsetlinewidth{0.818mm}
\pgfsetmiterlimit{7.0}
\pgfpathmoveto{\pgfqpoint{0.682cm}{0.671cm}}
\pgfpathlineto{\pgfqpoint{0.679cm}{1.418cm}}
\pgfusepath{stroke}
\end{pgfscope}
\pgfpathmoveto{\pgfqpoint{0.815cm}{1.399cm}}
\pgfpathcurveto{\pgfqpoint{0.815cm}{1.435cm}}{\pgfqpoint{0.801cm}{1.47cm}}{\pgfqpoint{0.775cm}{1.496cm}}
\pgfpathcurveto{\pgfqpoint{0.75cm}{1.521cm}}{\pgfqpoint{0.715cm}{1.536cm}}{\pgfqpoint{0.679cm}{1.536cm}}
\pgfpathcurveto{\pgfqpoint{0.643cm}{1.536cm}}{\pgfqpoint{0.608cm}{1.521cm}}{\pgfqpoint{0.582cm}{1.496cm}}
\pgfpathcurveto{\pgfqpoint{0.557cm}{1.47cm}}{\pgfqpoint{0.542cm}{1.435cm}}{\pgfqpoint{0.542cm}{1.399cm}}
\pgfpathcurveto{\pgfqpoint{0.542cm}{1.363cm}}{\pgfqpoint{0.557cm}{1.328cm}}{\pgfqpoint{0.582cm}{1.302cm}}
\pgfpathcurveto{\pgfqpoint{0.608cm}{1.276cm}}{\pgfqpoint{0.643cm}{1.262cm}}{\pgfqpoint{0.679cm}{1.262cm}}
\pgfpathcurveto{\pgfqpoint{0.715cm}{1.262cm}}{\pgfqpoint{0.75cm}{1.276cm}}{\pgfqpoint{0.775cm}{1.302cm}}
\pgfpathcurveto{\pgfqpoint{0.801cm}{1.328cm}}{\pgfqpoint{0.815cm}{1.363cm}}{\pgfqpoint{0.815cm}{1.399cm}}
\pgfusepath{fill}
\pgfpathmoveto{\pgfqpoint{1.345cm}{1.371cm}}
\pgfpathcurveto{\pgfqpoint{1.345cm}{1.408cm}}{\pgfqpoint{1.331cm}{1.442cm}}{\pgfqpoint{1.305cm}{1.468cm}}
\pgfpathcurveto{\pgfqpoint{1.28cm}{1.494cm}}{\pgfqpoint{1.245cm}{1.508cm}}{\pgfqpoint{1.209cm}{1.508cm}}
\pgfpathcurveto{\pgfqpoint{1.172cm}{1.508cm}}{\pgfqpoint{1.138cm}{1.494cm}}{\pgfqpoint{1.112cm}{1.468cm}}
\pgfpathcurveto{\pgfqpoint{1.087cm}{1.442cm}}{\pgfqpoint{1.072cm}{1.408cm}}{\pgfqpoint{1.072cm}{1.371cm}}
\pgfpathcurveto{\pgfqpoint{1.072cm}{1.335cm}}{\pgfqpoint{1.087cm}{1.3cm}}{\pgfqpoint{1.112cm}{1.274cm}}
\pgfpathcurveto{\pgfqpoint{1.138cm}{1.249cm}}{\pgfqpoint{1.172cm}{1.234cm}}{\pgfqpoint{1.209cm}{1.234cm}}
\pgfpathcurveto{\pgfqpoint{1.245cm}{1.234cm}}{\pgfqpoint{1.28cm}{1.249cm}}{\pgfqpoint{1.305cm}{1.274cm}}
\pgfpathcurveto{\pgfqpoint{1.331cm}{1.3cm}}{\pgfqpoint{1.345cm}{1.335cm}}{\pgfqpoint{1.345cm}{1.371cm}}
\pgfusepath{fill}
\begin{pgfscope}
\pgfsetdash{}{0cm}
\pgfsetlinewidth{0.818mm}
\pgfsetroundcap
\pgfsetmiterlimit{4.0}
\pgfpathmoveto{\pgfqpoint{0.682cm}{0.671cm}}
\pgfpathlineto{\pgfqpoint{0.682cm}{0.042cm}}
\pgfusepath{stroke}
\end{pgfscope}
\end{pgfscope}
\end{pgfscope}
\end{pgfscope}
\end{tikzpicture}}}(\phi+\psi))}=\rmm{6\UU_>X\succ(X^{\!\resizebox{0.6em}{!}{
\begin{tikzpicture}
\pgfpathmoveto{\pgfqpoint{0cm}{-0.035cm}}
\pgfpathlineto{\pgfqpoint{1.376cm}{-0.035cm}}
\pgfpathlineto{\pgfqpoint{1.376cm}{1.552cm}}
\pgfpathlineto{\pgfqpoint{0cm}{1.552cm}}
\pgfpathclose
\pgfusepath{clip}
\begin{pgfscope}
\begin{pgfscope}
\pgfpathmoveto{\pgfqpoint{0cm}{-0.035cm}}
\pgfpathlineto{\pgfqpoint{1.376cm}{-0.035cm}}
\pgfpathlineto{\pgfqpoint{1.376cm}{1.552cm}}
\pgfpathlineto{\pgfqpoint{0cm}{1.552cm}}
\pgfpathclose
\pgfusepath{clip}
\begin{pgfscope}
\begin{pgfscope}
\pgfsetdash{}{0cm}
\pgfsetlinewidth{0.818mm}
\pgfsetroundcap
\pgfsetroundjoin
\pgfsetmiterlimit{7.0}
\definecolor{eps2pgf_color}{gray}{0}\pgfsetstrokecolor{eps2pgf_color}\pgfsetfillcolor{eps2pgf_color}
\pgfpathmoveto{\pgfqpoint{0.117cm}{1.421cm}}
\pgfpathlineto{\pgfqpoint{0.682cm}{0.671cm}}
\pgfpathlineto{\pgfqpoint{1.246cm}{1.421cm}}
\pgfusepath{stroke}
\end{pgfscope}
\definecolor{eps2pgf_color}{gray}{0}\pgfsetstrokecolor{eps2pgf_color}\pgfsetfillcolor{eps2pgf_color}
\pgfpathmoveto{\pgfqpoint{0.273cm}{1.395cm}}
\pgfpathcurveto{\pgfqpoint{0.273cm}{1.432cm}}{\pgfqpoint{0.259cm}{1.467cm}}{\pgfqpoint{0.233cm}{1.492cm}}
\pgfpathcurveto{\pgfqpoint{0.207cm}{1.518cm}}{\pgfqpoint{0.173cm}{1.532cm}}{\pgfqpoint{0.137cm}{1.532cm}}
\pgfpathcurveto{\pgfqpoint{0.1cm}{1.532cm}}{\pgfqpoint{0.066cm}{1.518cm}}{\pgfqpoint{0.04cm}{1.492cm}}
\pgfpathcurveto{\pgfqpoint{0.014cm}{1.467cm}}{\pgfqpoint{0cm}{1.432cm}}{\pgfqpoint{0cm}{1.395cm}}
\pgfpathcurveto{\pgfqpoint{0cm}{1.359cm}}{\pgfqpoint{0.014cm}{1.324cm}}{\pgfqpoint{0.04cm}{1.299cm}}
\pgfpathcurveto{\pgfqpoint{0.066cm}{1.273cm}}{\pgfqpoint{0.1cm}{1.258cm}}{\pgfqpoint{0.137cm}{1.258cm}}
\pgfpathcurveto{\pgfqpoint{0.173cm}{1.258cm}}{\pgfqpoint{0.207cm}{1.273cm}}{\pgfqpoint{0.233cm}{1.299cm}}
\pgfpathcurveto{\pgfqpoint{0.259cm}{1.324cm}}{\pgfqpoint{0.273cm}{1.359cm}}{\pgfqpoint{0.273cm}{1.395cm}}
\pgfusepath{fill}
\begin{pgfscope}
\pgfsetdash{}{0cm}
\pgfsetlinewidth{0.818mm}
\pgfsetmiterlimit{7.0}
\pgfpathmoveto{\pgfqpoint{0.682cm}{0.671cm}}
\pgfpathlineto{\pgfqpoint{0.679cm}{1.418cm}}
\pgfusepath{stroke}
\end{pgfscope}
\pgfpathmoveto{\pgfqpoint{0.815cm}{1.399cm}}
\pgfpathcurveto{\pgfqpoint{0.815cm}{1.435cm}}{\pgfqpoint{0.801cm}{1.47cm}}{\pgfqpoint{0.775cm}{1.496cm}}
\pgfpathcurveto{\pgfqpoint{0.75cm}{1.521cm}}{\pgfqpoint{0.715cm}{1.536cm}}{\pgfqpoint{0.679cm}{1.536cm}}
\pgfpathcurveto{\pgfqpoint{0.643cm}{1.536cm}}{\pgfqpoint{0.608cm}{1.521cm}}{\pgfqpoint{0.582cm}{1.496cm}}
\pgfpathcurveto{\pgfqpoint{0.557cm}{1.47cm}}{\pgfqpoint{0.542cm}{1.435cm}}{\pgfqpoint{0.542cm}{1.399cm}}
\pgfpathcurveto{\pgfqpoint{0.542cm}{1.363cm}}{\pgfqpoint{0.557cm}{1.328cm}}{\pgfqpoint{0.582cm}{1.302cm}}
\pgfpathcurveto{\pgfqpoint{0.608cm}{1.276cm}}{\pgfqpoint{0.643cm}{1.262cm}}{\pgfqpoint{0.679cm}{1.262cm}}
\pgfpathcurveto{\pgfqpoint{0.715cm}{1.262cm}}{\pgfqpoint{0.75cm}{1.276cm}}{\pgfqpoint{0.775cm}{1.302cm}}
\pgfpathcurveto{\pgfqpoint{0.801cm}{1.328cm}}{\pgfqpoint{0.815cm}{1.363cm}}{\pgfqpoint{0.815cm}{1.399cm}}
\pgfusepath{fill}
\pgfpathmoveto{\pgfqpoint{1.345cm}{1.371cm}}
\pgfpathcurveto{\pgfqpoint{1.345cm}{1.408cm}}{\pgfqpoint{1.331cm}{1.442cm}}{\pgfqpoint{1.305cm}{1.468cm}}
\pgfpathcurveto{\pgfqpoint{1.28cm}{1.494cm}}{\pgfqpoint{1.245cm}{1.508cm}}{\pgfqpoint{1.209cm}{1.508cm}}
\pgfpathcurveto{\pgfqpoint{1.172cm}{1.508cm}}{\pgfqpoint{1.138cm}{1.494cm}}{\pgfqpoint{1.112cm}{1.468cm}}
\pgfpathcurveto{\pgfqpoint{1.087cm}{1.442cm}}{\pgfqpoint{1.072cm}{1.408cm}}{\pgfqpoint{1.072cm}{1.371cm}}
\pgfpathcurveto{\pgfqpoint{1.072cm}{1.335cm}}{\pgfqpoint{1.087cm}{1.3cm}}{\pgfqpoint{1.112cm}{1.274cm}}
\pgfpathcurveto{\pgfqpoint{1.138cm}{1.249cm}}{\pgfqpoint{1.172cm}{1.234cm}}{\pgfqpoint{1.209cm}{1.234cm}}
\pgfpathcurveto{\pgfqpoint{1.245cm}{1.234cm}}{\pgfqpoint{1.28cm}{1.249cm}}{\pgfqpoint{1.305cm}{1.274cm}}
\pgfpathcurveto{\pgfqpoint{1.331cm}{1.3cm}}{\pgfqpoint{1.345cm}{1.335cm}}{\pgfqpoint{1.345cm}{1.371cm}}
\pgfusepath{fill}
\begin{pgfscope}
\pgfsetdash{}{0cm}
\pgfsetlinewidth{0.818mm}
\pgfsetroundcap
\pgfsetmiterlimit{4.0}
\pgfpathmoveto{\pgfqpoint{0.682cm}{0.671cm}}
\pgfpathlineto{\pgfqpoint{0.682cm}{0.042cm}}
\pgfusepath{stroke}
\end{pgfscope}
\end{pgfscope}
\end{pgfscope}
\end{pgfscope}
\end{tikzpicture}}}(\phi+\psi))}+\rmb{6\UU_\leqslant X\succ(X^{\!\resizebox{0.6em}{!}{
\begin{tikzpicture}
\pgfpathmoveto{\pgfqpoint{0cm}{-0.035cm}}
\pgfpathlineto{\pgfqpoint{1.376cm}{-0.035cm}}
\pgfpathlineto{\pgfqpoint{1.376cm}{1.552cm}}
\pgfpathlineto{\pgfqpoint{0cm}{1.552cm}}
\pgfpathclose
\pgfusepath{clip}
\begin{pgfscope}
\begin{pgfscope}
\pgfpathmoveto{\pgfqpoint{0cm}{-0.035cm}}
\pgfpathlineto{\pgfqpoint{1.376cm}{-0.035cm}}
\pgfpathlineto{\pgfqpoint{1.376cm}{1.552cm}}
\pgfpathlineto{\pgfqpoint{0cm}{1.552cm}}
\pgfpathclose
\pgfusepath{clip}
\begin{pgfscope}
\begin{pgfscope}
\pgfsetdash{}{0cm}
\pgfsetlinewidth{0.818mm}
\pgfsetroundcap
\pgfsetroundjoin
\pgfsetmiterlimit{7.0}
\definecolor{eps2pgf_color}{gray}{0}\pgfsetstrokecolor{eps2pgf_color}\pgfsetfillcolor{eps2pgf_color}
\pgfpathmoveto{\pgfqpoint{0.117cm}{1.421cm}}
\pgfpathlineto{\pgfqpoint{0.682cm}{0.671cm}}
\pgfpathlineto{\pgfqpoint{1.246cm}{1.421cm}}
\pgfusepath{stroke}
\end{pgfscope}
\definecolor{eps2pgf_color}{gray}{0}\pgfsetstrokecolor{eps2pgf_color}\pgfsetfillcolor{eps2pgf_color}
\pgfpathmoveto{\pgfqpoint{0.273cm}{1.395cm}}
\pgfpathcurveto{\pgfqpoint{0.273cm}{1.432cm}}{\pgfqpoint{0.259cm}{1.467cm}}{\pgfqpoint{0.233cm}{1.492cm}}
\pgfpathcurveto{\pgfqpoint{0.207cm}{1.518cm}}{\pgfqpoint{0.173cm}{1.532cm}}{\pgfqpoint{0.137cm}{1.532cm}}
\pgfpathcurveto{\pgfqpoint{0.1cm}{1.532cm}}{\pgfqpoint{0.066cm}{1.518cm}}{\pgfqpoint{0.04cm}{1.492cm}}
\pgfpathcurveto{\pgfqpoint{0.014cm}{1.467cm}}{\pgfqpoint{0cm}{1.432cm}}{\pgfqpoint{0cm}{1.395cm}}
\pgfpathcurveto{\pgfqpoint{0cm}{1.359cm}}{\pgfqpoint{0.014cm}{1.324cm}}{\pgfqpoint{0.04cm}{1.299cm}}
\pgfpathcurveto{\pgfqpoint{0.066cm}{1.273cm}}{\pgfqpoint{0.1cm}{1.258cm}}{\pgfqpoint{0.137cm}{1.258cm}}
\pgfpathcurveto{\pgfqpoint{0.173cm}{1.258cm}}{\pgfqpoint{0.207cm}{1.273cm}}{\pgfqpoint{0.233cm}{1.299cm}}
\pgfpathcurveto{\pgfqpoint{0.259cm}{1.324cm}}{\pgfqpoint{0.273cm}{1.359cm}}{\pgfqpoint{0.273cm}{1.395cm}}
\pgfusepath{fill}
\begin{pgfscope}
\pgfsetdash{}{0cm}
\pgfsetlinewidth{0.818mm}
\pgfsetmiterlimit{7.0}
\pgfpathmoveto{\pgfqpoint{0.682cm}{0.671cm}}
\pgfpathlineto{\pgfqpoint{0.679cm}{1.418cm}}
\pgfusepath{stroke}
\end{pgfscope}
\pgfpathmoveto{\pgfqpoint{0.815cm}{1.399cm}}
\pgfpathcurveto{\pgfqpoint{0.815cm}{1.435cm}}{\pgfqpoint{0.801cm}{1.47cm}}{\pgfqpoint{0.775cm}{1.496cm}}
\pgfpathcurveto{\pgfqpoint{0.75cm}{1.521cm}}{\pgfqpoint{0.715cm}{1.536cm}}{\pgfqpoint{0.679cm}{1.536cm}}
\pgfpathcurveto{\pgfqpoint{0.643cm}{1.536cm}}{\pgfqpoint{0.608cm}{1.521cm}}{\pgfqpoint{0.582cm}{1.496cm}}
\pgfpathcurveto{\pgfqpoint{0.557cm}{1.47cm}}{\pgfqpoint{0.542cm}{1.435cm}}{\pgfqpoint{0.542cm}{1.399cm}}
\pgfpathcurveto{\pgfqpoint{0.542cm}{1.363cm}}{\pgfqpoint{0.557cm}{1.328cm}}{\pgfqpoint{0.582cm}{1.302cm}}
\pgfpathcurveto{\pgfqpoint{0.608cm}{1.276cm}}{\pgfqpoint{0.643cm}{1.262cm}}{\pgfqpoint{0.679cm}{1.262cm}}
\pgfpathcurveto{\pgfqpoint{0.715cm}{1.262cm}}{\pgfqpoint{0.75cm}{1.276cm}}{\pgfqpoint{0.775cm}{1.302cm}}
\pgfpathcurveto{\pgfqpoint{0.801cm}{1.328cm}}{\pgfqpoint{0.815cm}{1.363cm}}{\pgfqpoint{0.815cm}{1.399cm}}
\pgfusepath{fill}
\pgfpathmoveto{\pgfqpoint{1.345cm}{1.371cm}}
\pgfpathcurveto{\pgfqpoint{1.345cm}{1.408cm}}{\pgfqpoint{1.331cm}{1.442cm}}{\pgfqpoint{1.305cm}{1.468cm}}
\pgfpathcurveto{\pgfqpoint{1.28cm}{1.494cm}}{\pgfqpoint{1.245cm}{1.508cm}}{\pgfqpoint{1.209cm}{1.508cm}}
\pgfpathcurveto{\pgfqpoint{1.172cm}{1.508cm}}{\pgfqpoint{1.138cm}{1.494cm}}{\pgfqpoint{1.112cm}{1.468cm}}
\pgfpathcurveto{\pgfqpoint{1.087cm}{1.442cm}}{\pgfqpoint{1.072cm}{1.408cm}}{\pgfqpoint{1.072cm}{1.371cm}}
\pgfpathcurveto{\pgfqpoint{1.072cm}{1.335cm}}{\pgfqpoint{1.087cm}{1.3cm}}{\pgfqpoint{1.112cm}{1.274cm}}
\pgfpathcurveto{\pgfqpoint{1.138cm}{1.249cm}}{\pgfqpoint{1.172cm}{1.234cm}}{\pgfqpoint{1.209cm}{1.234cm}}
\pgfpathcurveto{\pgfqpoint{1.245cm}{1.234cm}}{\pgfqpoint{1.28cm}{1.249cm}}{\pgfqpoint{1.305cm}{1.274cm}}
\pgfpathcurveto{\pgfqpoint{1.331cm}{1.3cm}}{\pgfqpoint{1.345cm}{1.335cm}}{\pgfqpoint{1.345cm}{1.371cm}}
\pgfusepath{fill}
\begin{pgfscope}
\pgfsetdash{}{0cm}
\pgfsetlinewidth{0.818mm}
\pgfsetroundcap
\pgfsetmiterlimit{4.0}
\pgfpathmoveto{\pgfqpoint{0.682cm}{0.671cm}}
\pgfpathlineto{\pgfqpoint{0.682cm}{0.042cm}}
\pgfusepath{stroke}
\end{pgfscope}
\end{pgfscope}
\end{pgfscope}
\end{pgfscope}
\end{tikzpicture}}}(\phi+\psi))},
\end{align}
\begin{align}\label{eq:93}
\rmg{6X\prec(X^{\!\resizebox{0.6em}{!}{
\begin{tikzpicture}
\pgfpathmoveto{\pgfqpoint{0cm}{-0.035cm}}
\pgfpathlineto{\pgfqpoint{1.376cm}{-0.035cm}}
\pgfpathlineto{\pgfqpoint{1.376cm}{1.552cm}}
\pgfpathlineto{\pgfqpoint{0cm}{1.552cm}}
\pgfpathclose
\pgfusepath{clip}
\begin{pgfscope}
\begin{pgfscope}
\pgfpathmoveto{\pgfqpoint{0cm}{-0.035cm}}
\pgfpathlineto{\pgfqpoint{1.376cm}{-0.035cm}}
\pgfpathlineto{\pgfqpoint{1.376cm}{1.552cm}}
\pgfpathlineto{\pgfqpoint{0cm}{1.552cm}}
\pgfpathclose
\pgfusepath{clip}
\begin{pgfscope}
\begin{pgfscope}
\pgfsetdash{}{0cm}
\pgfsetlinewidth{0.818mm}
\pgfsetroundcap
\pgfsetroundjoin
\pgfsetmiterlimit{7.0}
\definecolor{eps2pgf_color}{gray}{0}\pgfsetstrokecolor{eps2pgf_color}\pgfsetfillcolor{eps2pgf_color}
\pgfpathmoveto{\pgfqpoint{0.117cm}{1.421cm}}
\pgfpathlineto{\pgfqpoint{0.682cm}{0.671cm}}
\pgfpathlineto{\pgfqpoint{1.246cm}{1.421cm}}
\pgfusepath{stroke}
\end{pgfscope}
\definecolor{eps2pgf_color}{gray}{0}\pgfsetstrokecolor{eps2pgf_color}\pgfsetfillcolor{eps2pgf_color}
\pgfpathmoveto{\pgfqpoint{0.273cm}{1.395cm}}
\pgfpathcurveto{\pgfqpoint{0.273cm}{1.432cm}}{\pgfqpoint{0.259cm}{1.467cm}}{\pgfqpoint{0.233cm}{1.492cm}}
\pgfpathcurveto{\pgfqpoint{0.207cm}{1.518cm}}{\pgfqpoint{0.173cm}{1.532cm}}{\pgfqpoint{0.137cm}{1.532cm}}
\pgfpathcurveto{\pgfqpoint{0.1cm}{1.532cm}}{\pgfqpoint{0.066cm}{1.518cm}}{\pgfqpoint{0.04cm}{1.492cm}}
\pgfpathcurveto{\pgfqpoint{0.014cm}{1.467cm}}{\pgfqpoint{0cm}{1.432cm}}{\pgfqpoint{0cm}{1.395cm}}
\pgfpathcurveto{\pgfqpoint{0cm}{1.359cm}}{\pgfqpoint{0.014cm}{1.324cm}}{\pgfqpoint{0.04cm}{1.299cm}}
\pgfpathcurveto{\pgfqpoint{0.066cm}{1.273cm}}{\pgfqpoint{0.1cm}{1.258cm}}{\pgfqpoint{0.137cm}{1.258cm}}
\pgfpathcurveto{\pgfqpoint{0.173cm}{1.258cm}}{\pgfqpoint{0.207cm}{1.273cm}}{\pgfqpoint{0.233cm}{1.299cm}}
\pgfpathcurveto{\pgfqpoint{0.259cm}{1.324cm}}{\pgfqpoint{0.273cm}{1.359cm}}{\pgfqpoint{0.273cm}{1.395cm}}
\pgfusepath{fill}
\begin{pgfscope}
\pgfsetdash{}{0cm}
\pgfsetlinewidth{0.818mm}
\pgfsetmiterlimit{7.0}
\pgfpathmoveto{\pgfqpoint{0.682cm}{0.671cm}}
\pgfpathlineto{\pgfqpoint{0.679cm}{1.418cm}}
\pgfusepath{stroke}
\end{pgfscope}
\pgfpathmoveto{\pgfqpoint{0.815cm}{1.399cm}}
\pgfpathcurveto{\pgfqpoint{0.815cm}{1.435cm}}{\pgfqpoint{0.801cm}{1.47cm}}{\pgfqpoint{0.775cm}{1.496cm}}
\pgfpathcurveto{\pgfqpoint{0.75cm}{1.521cm}}{\pgfqpoint{0.715cm}{1.536cm}}{\pgfqpoint{0.679cm}{1.536cm}}
\pgfpathcurveto{\pgfqpoint{0.643cm}{1.536cm}}{\pgfqpoint{0.608cm}{1.521cm}}{\pgfqpoint{0.582cm}{1.496cm}}
\pgfpathcurveto{\pgfqpoint{0.557cm}{1.47cm}}{\pgfqpoint{0.542cm}{1.435cm}}{\pgfqpoint{0.542cm}{1.399cm}}
\pgfpathcurveto{\pgfqpoint{0.542cm}{1.363cm}}{\pgfqpoint{0.557cm}{1.328cm}}{\pgfqpoint{0.582cm}{1.302cm}}
\pgfpathcurveto{\pgfqpoint{0.608cm}{1.276cm}}{\pgfqpoint{0.643cm}{1.262cm}}{\pgfqpoint{0.679cm}{1.262cm}}
\pgfpathcurveto{\pgfqpoint{0.715cm}{1.262cm}}{\pgfqpoint{0.75cm}{1.276cm}}{\pgfqpoint{0.775cm}{1.302cm}}
\pgfpathcurveto{\pgfqpoint{0.801cm}{1.328cm}}{\pgfqpoint{0.815cm}{1.363cm}}{\pgfqpoint{0.815cm}{1.399cm}}
\pgfusepath{fill}
\pgfpathmoveto{\pgfqpoint{1.345cm}{1.371cm}}
\pgfpathcurveto{\pgfqpoint{1.345cm}{1.408cm}}{\pgfqpoint{1.331cm}{1.442cm}}{\pgfqpoint{1.305cm}{1.468cm}}
\pgfpathcurveto{\pgfqpoint{1.28cm}{1.494cm}}{\pgfqpoint{1.245cm}{1.508cm}}{\pgfqpoint{1.209cm}{1.508cm}}
\pgfpathcurveto{\pgfqpoint{1.172cm}{1.508cm}}{\pgfqpoint{1.138cm}{1.494cm}}{\pgfqpoint{1.112cm}{1.468cm}}
\pgfpathcurveto{\pgfqpoint{1.087cm}{1.442cm}}{\pgfqpoint{1.072cm}{1.408cm}}{\pgfqpoint{1.072cm}{1.371cm}}
\pgfpathcurveto{\pgfqpoint{1.072cm}{1.335cm}}{\pgfqpoint{1.087cm}{1.3cm}}{\pgfqpoint{1.112cm}{1.274cm}}
\pgfpathcurveto{\pgfqpoint{1.138cm}{1.249cm}}{\pgfqpoint{1.172cm}{1.234cm}}{\pgfqpoint{1.209cm}{1.234cm}}
\pgfpathcurveto{\pgfqpoint{1.245cm}{1.234cm}}{\pgfqpoint{1.28cm}{1.249cm}}{\pgfqpoint{1.305cm}{1.274cm}}
\pgfpathcurveto{\pgfqpoint{1.331cm}{1.3cm}}{\pgfqpoint{1.345cm}{1.335cm}}{\pgfqpoint{1.345cm}{1.371cm}}
\pgfusepath{fill}
\begin{pgfscope}
\pgfsetdash{}{0cm}
\pgfsetlinewidth{0.818mm}
\pgfsetroundcap
\pgfsetmiterlimit{4.0}
\pgfpathmoveto{\pgfqpoint{0.682cm}{0.671cm}}
\pgfpathlineto{\pgfqpoint{0.682cm}{0.042cm}}
\pgfusepath{stroke}
\end{pgfscope}
\end{pgfscope}
\end{pgfscope}
\end{pgfscope}
\end{tikzpicture}}}(\phi+\psi))}=\rmm{6\UU_> X\prec(X^{\!\resizebox{0.6em}{!}{
\begin{tikzpicture}
\pgfpathmoveto{\pgfqpoint{0cm}{-0.035cm}}
\pgfpathlineto{\pgfqpoint{1.376cm}{-0.035cm}}
\pgfpathlineto{\pgfqpoint{1.376cm}{1.552cm}}
\pgfpathlineto{\pgfqpoint{0cm}{1.552cm}}
\pgfpathclose
\pgfusepath{clip}
\begin{pgfscope}
\begin{pgfscope}
\pgfpathmoveto{\pgfqpoint{0cm}{-0.035cm}}
\pgfpathlineto{\pgfqpoint{1.376cm}{-0.035cm}}
\pgfpathlineto{\pgfqpoint{1.376cm}{1.552cm}}
\pgfpathlineto{\pgfqpoint{0cm}{1.552cm}}
\pgfpathclose
\pgfusepath{clip}
\begin{pgfscope}
\begin{pgfscope}
\pgfsetdash{}{0cm}
\pgfsetlinewidth{0.818mm}
\pgfsetroundcap
\pgfsetroundjoin
\pgfsetmiterlimit{7.0}
\definecolor{eps2pgf_color}{gray}{0}\pgfsetstrokecolor{eps2pgf_color}\pgfsetfillcolor{eps2pgf_color}
\pgfpathmoveto{\pgfqpoint{0.117cm}{1.421cm}}
\pgfpathlineto{\pgfqpoint{0.682cm}{0.671cm}}
\pgfpathlineto{\pgfqpoint{1.246cm}{1.421cm}}
\pgfusepath{stroke}
\end{pgfscope}
\definecolor{eps2pgf_color}{gray}{0}\pgfsetstrokecolor{eps2pgf_color}\pgfsetfillcolor{eps2pgf_color}
\pgfpathmoveto{\pgfqpoint{0.273cm}{1.395cm}}
\pgfpathcurveto{\pgfqpoint{0.273cm}{1.432cm}}{\pgfqpoint{0.259cm}{1.467cm}}{\pgfqpoint{0.233cm}{1.492cm}}
\pgfpathcurveto{\pgfqpoint{0.207cm}{1.518cm}}{\pgfqpoint{0.173cm}{1.532cm}}{\pgfqpoint{0.137cm}{1.532cm}}
\pgfpathcurveto{\pgfqpoint{0.1cm}{1.532cm}}{\pgfqpoint{0.066cm}{1.518cm}}{\pgfqpoint{0.04cm}{1.492cm}}
\pgfpathcurveto{\pgfqpoint{0.014cm}{1.467cm}}{\pgfqpoint{0cm}{1.432cm}}{\pgfqpoint{0cm}{1.395cm}}
\pgfpathcurveto{\pgfqpoint{0cm}{1.359cm}}{\pgfqpoint{0.014cm}{1.324cm}}{\pgfqpoint{0.04cm}{1.299cm}}
\pgfpathcurveto{\pgfqpoint{0.066cm}{1.273cm}}{\pgfqpoint{0.1cm}{1.258cm}}{\pgfqpoint{0.137cm}{1.258cm}}
\pgfpathcurveto{\pgfqpoint{0.173cm}{1.258cm}}{\pgfqpoint{0.207cm}{1.273cm}}{\pgfqpoint{0.233cm}{1.299cm}}
\pgfpathcurveto{\pgfqpoint{0.259cm}{1.324cm}}{\pgfqpoint{0.273cm}{1.359cm}}{\pgfqpoint{0.273cm}{1.395cm}}
\pgfusepath{fill}
\begin{pgfscope}
\pgfsetdash{}{0cm}
\pgfsetlinewidth{0.818mm}
\pgfsetmiterlimit{7.0}
\pgfpathmoveto{\pgfqpoint{0.682cm}{0.671cm}}
\pgfpathlineto{\pgfqpoint{0.679cm}{1.418cm}}
\pgfusepath{stroke}
\end{pgfscope}
\pgfpathmoveto{\pgfqpoint{0.815cm}{1.399cm}}
\pgfpathcurveto{\pgfqpoint{0.815cm}{1.435cm}}{\pgfqpoint{0.801cm}{1.47cm}}{\pgfqpoint{0.775cm}{1.496cm}}
\pgfpathcurveto{\pgfqpoint{0.75cm}{1.521cm}}{\pgfqpoint{0.715cm}{1.536cm}}{\pgfqpoint{0.679cm}{1.536cm}}
\pgfpathcurveto{\pgfqpoint{0.643cm}{1.536cm}}{\pgfqpoint{0.608cm}{1.521cm}}{\pgfqpoint{0.582cm}{1.496cm}}
\pgfpathcurveto{\pgfqpoint{0.557cm}{1.47cm}}{\pgfqpoint{0.542cm}{1.435cm}}{\pgfqpoint{0.542cm}{1.399cm}}
\pgfpathcurveto{\pgfqpoint{0.542cm}{1.363cm}}{\pgfqpoint{0.557cm}{1.328cm}}{\pgfqpoint{0.582cm}{1.302cm}}
\pgfpathcurveto{\pgfqpoint{0.608cm}{1.276cm}}{\pgfqpoint{0.643cm}{1.262cm}}{\pgfqpoint{0.679cm}{1.262cm}}
\pgfpathcurveto{\pgfqpoint{0.715cm}{1.262cm}}{\pgfqpoint{0.75cm}{1.276cm}}{\pgfqpoint{0.775cm}{1.302cm}}
\pgfpathcurveto{\pgfqpoint{0.801cm}{1.328cm}}{\pgfqpoint{0.815cm}{1.363cm}}{\pgfqpoint{0.815cm}{1.399cm}}
\pgfusepath{fill}
\pgfpathmoveto{\pgfqpoint{1.345cm}{1.371cm}}
\pgfpathcurveto{\pgfqpoint{1.345cm}{1.408cm}}{\pgfqpoint{1.331cm}{1.442cm}}{\pgfqpoint{1.305cm}{1.468cm}}
\pgfpathcurveto{\pgfqpoint{1.28cm}{1.494cm}}{\pgfqpoint{1.245cm}{1.508cm}}{\pgfqpoint{1.209cm}{1.508cm}}
\pgfpathcurveto{\pgfqpoint{1.172cm}{1.508cm}}{\pgfqpoint{1.138cm}{1.494cm}}{\pgfqpoint{1.112cm}{1.468cm}}
\pgfpathcurveto{\pgfqpoint{1.087cm}{1.442cm}}{\pgfqpoint{1.072cm}{1.408cm}}{\pgfqpoint{1.072cm}{1.371cm}}
\pgfpathcurveto{\pgfqpoint{1.072cm}{1.335cm}}{\pgfqpoint{1.087cm}{1.3cm}}{\pgfqpoint{1.112cm}{1.274cm}}
\pgfpathcurveto{\pgfqpoint{1.138cm}{1.249cm}}{\pgfqpoint{1.172cm}{1.234cm}}{\pgfqpoint{1.209cm}{1.234cm}}
\pgfpathcurveto{\pgfqpoint{1.245cm}{1.234cm}}{\pgfqpoint{1.28cm}{1.249cm}}{\pgfqpoint{1.305cm}{1.274cm}}
\pgfpathcurveto{\pgfqpoint{1.331cm}{1.3cm}}{\pgfqpoint{1.345cm}{1.335cm}}{\pgfqpoint{1.345cm}{1.371cm}}
\pgfusepath{fill}
\begin{pgfscope}
\pgfsetdash{}{0cm}
\pgfsetlinewidth{0.818mm}
\pgfsetroundcap
\pgfsetmiterlimit{4.0}
\pgfpathmoveto{\pgfqpoint{0.682cm}{0.671cm}}
\pgfpathlineto{\pgfqpoint{0.682cm}{0.042cm}}
\pgfusepath{stroke}
\end{pgfscope}
\end{pgfscope}
\end{pgfscope}
\end{pgfscope}
\end{tikzpicture}}}(\phi+\psi))}+\rmb{6\UU_\leqslant X\prec(X^{\!\resizebox{0.6em}{!}{
\begin{tikzpicture}
\pgfpathmoveto{\pgfqpoint{0cm}{-0.035cm}}
\pgfpathlineto{\pgfqpoint{1.376cm}{-0.035cm}}
\pgfpathlineto{\pgfqpoint{1.376cm}{1.552cm}}
\pgfpathlineto{\pgfqpoint{0cm}{1.552cm}}
\pgfpathclose
\pgfusepath{clip}
\begin{pgfscope}
\begin{pgfscope}
\pgfpathmoveto{\pgfqpoint{0cm}{-0.035cm}}
\pgfpathlineto{\pgfqpoint{1.376cm}{-0.035cm}}
\pgfpathlineto{\pgfqpoint{1.376cm}{1.552cm}}
\pgfpathlineto{\pgfqpoint{0cm}{1.552cm}}
\pgfpathclose
\pgfusepath{clip}
\begin{pgfscope}
\begin{pgfscope}
\pgfsetdash{}{0cm}
\pgfsetlinewidth{0.818mm}
\pgfsetroundcap
\pgfsetroundjoin
\pgfsetmiterlimit{7.0}
\definecolor{eps2pgf_color}{gray}{0}\pgfsetstrokecolor{eps2pgf_color}\pgfsetfillcolor{eps2pgf_color}
\pgfpathmoveto{\pgfqpoint{0.117cm}{1.421cm}}
\pgfpathlineto{\pgfqpoint{0.682cm}{0.671cm}}
\pgfpathlineto{\pgfqpoint{1.246cm}{1.421cm}}
\pgfusepath{stroke}
\end{pgfscope}
\definecolor{eps2pgf_color}{gray}{0}\pgfsetstrokecolor{eps2pgf_color}\pgfsetfillcolor{eps2pgf_color}
\pgfpathmoveto{\pgfqpoint{0.273cm}{1.395cm}}
\pgfpathcurveto{\pgfqpoint{0.273cm}{1.432cm}}{\pgfqpoint{0.259cm}{1.467cm}}{\pgfqpoint{0.233cm}{1.492cm}}
\pgfpathcurveto{\pgfqpoint{0.207cm}{1.518cm}}{\pgfqpoint{0.173cm}{1.532cm}}{\pgfqpoint{0.137cm}{1.532cm}}
\pgfpathcurveto{\pgfqpoint{0.1cm}{1.532cm}}{\pgfqpoint{0.066cm}{1.518cm}}{\pgfqpoint{0.04cm}{1.492cm}}
\pgfpathcurveto{\pgfqpoint{0.014cm}{1.467cm}}{\pgfqpoint{0cm}{1.432cm}}{\pgfqpoint{0cm}{1.395cm}}
\pgfpathcurveto{\pgfqpoint{0cm}{1.359cm}}{\pgfqpoint{0.014cm}{1.324cm}}{\pgfqpoint{0.04cm}{1.299cm}}
\pgfpathcurveto{\pgfqpoint{0.066cm}{1.273cm}}{\pgfqpoint{0.1cm}{1.258cm}}{\pgfqpoint{0.137cm}{1.258cm}}
\pgfpathcurveto{\pgfqpoint{0.173cm}{1.258cm}}{\pgfqpoint{0.207cm}{1.273cm}}{\pgfqpoint{0.233cm}{1.299cm}}
\pgfpathcurveto{\pgfqpoint{0.259cm}{1.324cm}}{\pgfqpoint{0.273cm}{1.359cm}}{\pgfqpoint{0.273cm}{1.395cm}}
\pgfusepath{fill}
\begin{pgfscope}
\pgfsetdash{}{0cm}
\pgfsetlinewidth{0.818mm}
\pgfsetmiterlimit{7.0}
\pgfpathmoveto{\pgfqpoint{0.682cm}{0.671cm}}
\pgfpathlineto{\pgfqpoint{0.679cm}{1.418cm}}
\pgfusepath{stroke}
\end{pgfscope}
\pgfpathmoveto{\pgfqpoint{0.815cm}{1.399cm}}
\pgfpathcurveto{\pgfqpoint{0.815cm}{1.435cm}}{\pgfqpoint{0.801cm}{1.47cm}}{\pgfqpoint{0.775cm}{1.496cm}}
\pgfpathcurveto{\pgfqpoint{0.75cm}{1.521cm}}{\pgfqpoint{0.715cm}{1.536cm}}{\pgfqpoint{0.679cm}{1.536cm}}
\pgfpathcurveto{\pgfqpoint{0.643cm}{1.536cm}}{\pgfqpoint{0.608cm}{1.521cm}}{\pgfqpoint{0.582cm}{1.496cm}}
\pgfpathcurveto{\pgfqpoint{0.557cm}{1.47cm}}{\pgfqpoint{0.542cm}{1.435cm}}{\pgfqpoint{0.542cm}{1.399cm}}
\pgfpathcurveto{\pgfqpoint{0.542cm}{1.363cm}}{\pgfqpoint{0.557cm}{1.328cm}}{\pgfqpoint{0.582cm}{1.302cm}}
\pgfpathcurveto{\pgfqpoint{0.608cm}{1.276cm}}{\pgfqpoint{0.643cm}{1.262cm}}{\pgfqpoint{0.679cm}{1.262cm}}
\pgfpathcurveto{\pgfqpoint{0.715cm}{1.262cm}}{\pgfqpoint{0.75cm}{1.276cm}}{\pgfqpoint{0.775cm}{1.302cm}}
\pgfpathcurveto{\pgfqpoint{0.801cm}{1.328cm}}{\pgfqpoint{0.815cm}{1.363cm}}{\pgfqpoint{0.815cm}{1.399cm}}
\pgfusepath{fill}
\pgfpathmoveto{\pgfqpoint{1.345cm}{1.371cm}}
\pgfpathcurveto{\pgfqpoint{1.345cm}{1.408cm}}{\pgfqpoint{1.331cm}{1.442cm}}{\pgfqpoint{1.305cm}{1.468cm}}
\pgfpathcurveto{\pgfqpoint{1.28cm}{1.494cm}}{\pgfqpoint{1.245cm}{1.508cm}}{\pgfqpoint{1.209cm}{1.508cm}}
\pgfpathcurveto{\pgfqpoint{1.172cm}{1.508cm}}{\pgfqpoint{1.138cm}{1.494cm}}{\pgfqpoint{1.112cm}{1.468cm}}
\pgfpathcurveto{\pgfqpoint{1.087cm}{1.442cm}}{\pgfqpoint{1.072cm}{1.408cm}}{\pgfqpoint{1.072cm}{1.371cm}}
\pgfpathcurveto{\pgfqpoint{1.072cm}{1.335cm}}{\pgfqpoint{1.087cm}{1.3cm}}{\pgfqpoint{1.112cm}{1.274cm}}
\pgfpathcurveto{\pgfqpoint{1.138cm}{1.249cm}}{\pgfqpoint{1.172cm}{1.234cm}}{\pgfqpoint{1.209cm}{1.234cm}}
\pgfpathcurveto{\pgfqpoint{1.245cm}{1.234cm}}{\pgfqpoint{1.28cm}{1.249cm}}{\pgfqpoint{1.305cm}{1.274cm}}
\pgfpathcurveto{\pgfqpoint{1.331cm}{1.3cm}}{\pgfqpoint{1.345cm}{1.335cm}}{\pgfqpoint{1.345cm}{1.371cm}}
\pgfusepath{fill}
\begin{pgfscope}
\pgfsetdash{}{0cm}
\pgfsetlinewidth{0.818mm}
\pgfsetroundcap
\pgfsetmiterlimit{4.0}
\pgfpathmoveto{\pgfqpoint{0.682cm}{0.671cm}}
\pgfpathlineto{\pgfqpoint{0.682cm}{0.042cm}}
\pgfusepath{stroke}
\end{pgfscope}
\end{pgfscope}
\end{pgfscope}
\end{pgfscope}
\end{tikzpicture}}}(\phi+\psi))},
\end{align}
\begin{align*}
\rmg{6(\phi+\psi)X^{\!\resizebox{!}{.8em}{
\begin{tikzpicture}
\pgfpathmoveto{\pgfqpoint{0cm}{-0.035cm}}
\pgfpathlineto{\pgfqpoint{1.976cm}{-0.035cm}}
\pgfpathlineto{\pgfqpoint{1.976cm}{1.94cm}}
\pgfpathlineto{\pgfqpoint{0cm}{1.94cm}}
\pgfpathclose
\pgfusepath{clip}
\begin{pgfscope}
\begin{pgfscope}
\pgfpathmoveto{\pgfqpoint{0cm}{-0.035cm}}
\pgfpathlineto{\pgfqpoint{1.976cm}{-0.035cm}}
\pgfpathlineto{\pgfqpoint{1.976cm}{1.94cm}}
\pgfpathlineto{\pgfqpoint{0cm}{1.94cm}}
\pgfpathclose
\pgfusepath{clip}
\begin{pgfscope}
\begin{pgfscope}
\pgfsetdash{}{0cm}
\pgfsetlinewidth{0.818mm}
\pgfsetroundcap
\pgfsetroundjoin
\pgfsetmiterlimit{7.0}
\definecolor{eps2pgf_color}{gray}{0}\pgfsetstrokecolor{eps2pgf_color}\pgfsetfillcolor{eps2pgf_color}
\pgfpathmoveto{\pgfqpoint{0.117cm}{1.815cm}}
\pgfpathlineto{\pgfqpoint{0.682cm}{1.065cm}}
\pgfpathlineto{\pgfqpoint{1.246cm}{1.815cm}}
\pgfusepath{stroke}
\end{pgfscope}
\definecolor{eps2pgf_color}{gray}{0}\pgfsetstrokecolor{eps2pgf_color}\pgfsetfillcolor{eps2pgf_color}
\pgfpathmoveto{\pgfqpoint{0.273cm}{1.789cm}}
\pgfpathcurveto{\pgfqpoint{0.273cm}{1.825cm}}{\pgfqpoint{0.259cm}{1.86cm}}{\pgfqpoint{0.233cm}{1.886cm}}
\pgfpathcurveto{\pgfqpoint{0.207cm}{1.912cm}}{\pgfqpoint{0.173cm}{1.926cm}}{\pgfqpoint{0.137cm}{1.926cm}}
\pgfpathcurveto{\pgfqpoint{0.1cm}{1.926cm}}{\pgfqpoint{0.066cm}{1.912cm}}{\pgfqpoint{0.04cm}{1.886cm}}
\pgfpathcurveto{\pgfqpoint{0.014cm}{1.86cm}}{\pgfqpoint{0cm}{1.825cm}}{\pgfqpoint{0cm}{1.789cm}}
\pgfpathcurveto{\pgfqpoint{0cm}{1.753cm}}{\pgfqpoint{0.014cm}{1.718cm}}{\pgfqpoint{0.04cm}{1.692cm}}
\pgfpathcurveto{\pgfqpoint{0.066cm}{1.667cm}}{\pgfqpoint{0.1cm}{1.652cm}}{\pgfqpoint{0.137cm}{1.652cm}}
\pgfpathcurveto{\pgfqpoint{0.173cm}{1.652cm}}{\pgfqpoint{0.207cm}{1.667cm}}{\pgfqpoint{0.233cm}{1.692cm}}
\pgfpathcurveto{\pgfqpoint{0.259cm}{1.718cm}}{\pgfqpoint{0.273cm}{1.753cm}}{\pgfqpoint{0.273cm}{1.789cm}}
\pgfusepath{fill}
\begin{pgfscope}
\pgfsetdash{}{0cm}
\pgfsetlinewidth{0.818mm}
\pgfsetmiterlimit{7.0}
\pgfpathmoveto{\pgfqpoint{0.682cm}{1.065cm}}
\pgfpathlineto{\pgfqpoint{0.679cm}{1.812cm}}
\pgfusepath{stroke}
\end{pgfscope}
\pgfpathmoveto{\pgfqpoint{0.815cm}{1.793cm}}
\pgfpathcurveto{\pgfqpoint{0.815cm}{1.829cm}}{\pgfqpoint{0.801cm}{1.864cm}}{\pgfqpoint{0.775cm}{1.89cm}}
\pgfpathcurveto{\pgfqpoint{0.75cm}{1.915cm}}{\pgfqpoint{0.715cm}{1.93cm}}{\pgfqpoint{0.679cm}{1.93cm}}
\pgfpathcurveto{\pgfqpoint{0.643cm}{1.93cm}}{\pgfqpoint{0.608cm}{1.915cm}}{\pgfqpoint{0.582cm}{1.89cm}}
\pgfpathcurveto{\pgfqpoint{0.557cm}{1.864cm}}{\pgfqpoint{0.542cm}{1.829cm}}{\pgfqpoint{0.542cm}{1.793cm}}
\pgfpathcurveto{\pgfqpoint{0.542cm}{1.756cm}}{\pgfqpoint{0.557cm}{1.722cm}}{\pgfqpoint{0.582cm}{1.696cm}}
\pgfpathcurveto{\pgfqpoint{0.608cm}{1.67cm}}{\pgfqpoint{0.643cm}{1.656cm}}{\pgfqpoint{0.679cm}{1.656cm}}
\pgfpathcurveto{\pgfqpoint{0.715cm}{1.656cm}}{\pgfqpoint{0.75cm}{1.67cm}}{\pgfqpoint{0.775cm}{1.696cm}}
\pgfpathcurveto{\pgfqpoint{0.801cm}{1.722cm}}{\pgfqpoint{0.815cm}{1.756cm}}{\pgfqpoint{0.815cm}{1.793cm}}
\pgfusepath{fill}
\pgfpathmoveto{\pgfqpoint{1.345cm}{1.765cm}}
\pgfpathcurveto{\pgfqpoint{1.345cm}{1.801cm}}{\pgfqpoint{1.331cm}{1.836cm}}{\pgfqpoint{1.305cm}{1.862cm}}
\pgfpathcurveto{\pgfqpoint{1.28cm}{1.887cm}}{\pgfqpoint{1.245cm}{1.902cm}}{\pgfqpoint{1.209cm}{1.902cm}}
\pgfpathcurveto{\pgfqpoint{1.172cm}{1.902cm}}{\pgfqpoint{1.138cm}{1.887cm}}{\pgfqpoint{1.112cm}{1.862cm}}
\pgfpathcurveto{\pgfqpoint{1.087cm}{1.836cm}}{\pgfqpoint{1.072cm}{1.801cm}}{\pgfqpoint{1.072cm}{1.765cm}}
\pgfpathcurveto{\pgfqpoint{1.072cm}{1.728cm}}{\pgfqpoint{1.087cm}{1.694cm}}{\pgfqpoint{1.112cm}{1.668cm}}
\pgfpathcurveto{\pgfqpoint{1.138cm}{1.642cm}}{\pgfqpoint{1.172cm}{1.628cm}}{\pgfqpoint{1.209cm}{1.628cm}}
\pgfpathcurveto{\pgfqpoint{1.245cm}{1.628cm}}{\pgfqpoint{1.28cm}{1.642cm}}{\pgfqpoint{1.305cm}{1.668cm}}
\pgfpathcurveto{\pgfqpoint{1.331cm}{1.694cm}}{\pgfqpoint{1.345cm}{1.728cm}}{\pgfqpoint{1.345cm}{1.765cm}}
\pgfusepath{fill}
\begin{pgfscope}
\pgfsetdash{}{0cm}
\pgfsetlinewidth{0.818mm}
\pgfsetroundcap
\pgfsetroundjoin
\pgfsetmiterlimit{7.0}
\pgfpathmoveto{\pgfqpoint{0.682cm}{1.065cm}}
\pgfpathlineto{\pgfqpoint{1.246cm}{0.315cm}}
\pgfpathlineto{\pgfqpoint{1.811cm}{1.065cm}}
\pgfusepath{stroke}
\end{pgfscope}
\pgfpathmoveto{\pgfqpoint{1.948cm}{1.065cm}}
\pgfpathcurveto{\pgfqpoint{1.948cm}{1.101cm}}{\pgfqpoint{1.933cm}{1.136cm}}{\pgfqpoint{1.907cm}{1.162cm}}
\pgfpathcurveto{\pgfqpoint{1.882cm}{1.187cm}}{\pgfqpoint{1.847cm}{1.202cm}}{\pgfqpoint{1.811cm}{1.202cm}}
\pgfpathcurveto{\pgfqpoint{1.775cm}{1.202cm}}{\pgfqpoint{1.74cm}{1.187cm}}{\pgfqpoint{1.714cm}{1.162cm}}
\pgfpathcurveto{\pgfqpoint{1.689cm}{1.136cm}}{\pgfqpoint{1.674cm}{1.101cm}}{\pgfqpoint{1.674cm}{1.065cm}}
\pgfpathcurveto{\pgfqpoint{1.674cm}{1.029cm}}{\pgfqpoint{1.689cm}{0.994cm}}{\pgfqpoint{1.714cm}{0.968cm}}
\pgfpathcurveto{\pgfqpoint{1.74cm}{0.942cm}}{\pgfqpoint{1.775cm}{0.928cm}}{\pgfqpoint{1.811cm}{0.928cm}}
\pgfpathcurveto{\pgfqpoint{1.847cm}{0.928cm}}{\pgfqpoint{1.882cm}{0.942cm}}{\pgfqpoint{1.907cm}{0.968cm}}
\pgfpathcurveto{\pgfqpoint{1.933cm}{0.994cm}}{\pgfqpoint{1.948cm}{1.029cm}}{\pgfqpoint{1.948cm}{1.065cm}}
\pgfusepath{fill}
\begin{pgfscope}
\pgfsetdash{}{0cm}
\pgfsetlinewidth{0.818mm}
\pgfsetmiterlimit{4.0}
\pgfpathmoveto{\pgfqpoint{1.383cm}{0.178cm}}
\pgfpathcurveto{\pgfqpoint{1.383cm}{0.214cm}}{\pgfqpoint{1.369cm}{0.249cm}}{\pgfqpoint{1.343cm}{0.275cm}}
\pgfpathcurveto{\pgfqpoint{1.317cm}{0.3cm}}{\pgfqpoint{1.283cm}{0.315cm}}{\pgfqpoint{1.246cm}{0.315cm}}
\pgfpathcurveto{\pgfqpoint{1.21cm}{0.315cm}}{\pgfqpoint{1.175cm}{0.3cm}}{\pgfqpoint{1.15cm}{0.275cm}}
\pgfpathcurveto{\pgfqpoint{1.124cm}{0.249cm}}{\pgfqpoint{1.11cm}{0.214cm}}{\pgfqpoint{1.11cm}{0.178cm}}
\pgfpathcurveto{\pgfqpoint{1.11cm}{0.141cm}}{\pgfqpoint{1.124cm}{0.107cm}}{\pgfqpoint{1.15cm}{0.081cm}}
\pgfpathcurveto{\pgfqpoint{1.175cm}{0.055cm}}{\pgfqpoint{1.21cm}{0.041cm}}{\pgfqpoint{1.246cm}{0.041cm}}
\pgfpathcurveto{\pgfqpoint{1.283cm}{0.041cm}}{\pgfqpoint{1.317cm}{0.055cm}}{\pgfqpoint{1.343cm}{0.081cm}}
\pgfpathcurveto{\pgfqpoint{1.369cm}{0.107cm}}{\pgfqpoint{1.383cm}{0.141cm}}{\pgfqpoint{1.383cm}{0.178cm}}
\pgfusepath{stroke}
\end{pgfscope}
\end{pgfscope}
\end{pgfscope}
\end{pgfscope}
\end{tikzpicture}}}}=\rmm{6(\phi+\psi)\prec\UU_{>}X^{\!\resizebox{!}{.8em}{
\begin{tikzpicture}
\pgfpathmoveto{\pgfqpoint{0cm}{-0.035cm}}
\pgfpathlineto{\pgfqpoint{1.976cm}{-0.035cm}}
\pgfpathlineto{\pgfqpoint{1.976cm}{1.94cm}}
\pgfpathlineto{\pgfqpoint{0cm}{1.94cm}}
\pgfpathclose
\pgfusepath{clip}
\begin{pgfscope}
\begin{pgfscope}
\pgfpathmoveto{\pgfqpoint{0cm}{-0.035cm}}
\pgfpathlineto{\pgfqpoint{1.976cm}{-0.035cm}}
\pgfpathlineto{\pgfqpoint{1.976cm}{1.94cm}}
\pgfpathlineto{\pgfqpoint{0cm}{1.94cm}}
\pgfpathclose
\pgfusepath{clip}
\begin{pgfscope}
\begin{pgfscope}
\pgfsetdash{}{0cm}
\pgfsetlinewidth{0.818mm}
\pgfsetroundcap
\pgfsetroundjoin
\pgfsetmiterlimit{7.0}
\definecolor{eps2pgf_color}{gray}{0}\pgfsetstrokecolor{eps2pgf_color}\pgfsetfillcolor{eps2pgf_color}
\pgfpathmoveto{\pgfqpoint{0.117cm}{1.815cm}}
\pgfpathlineto{\pgfqpoint{0.682cm}{1.065cm}}
\pgfpathlineto{\pgfqpoint{1.246cm}{1.815cm}}
\pgfusepath{stroke}
\end{pgfscope}
\definecolor{eps2pgf_color}{gray}{0}\pgfsetstrokecolor{eps2pgf_color}\pgfsetfillcolor{eps2pgf_color}
\pgfpathmoveto{\pgfqpoint{0.273cm}{1.789cm}}
\pgfpathcurveto{\pgfqpoint{0.273cm}{1.825cm}}{\pgfqpoint{0.259cm}{1.86cm}}{\pgfqpoint{0.233cm}{1.886cm}}
\pgfpathcurveto{\pgfqpoint{0.207cm}{1.912cm}}{\pgfqpoint{0.173cm}{1.926cm}}{\pgfqpoint{0.137cm}{1.926cm}}
\pgfpathcurveto{\pgfqpoint{0.1cm}{1.926cm}}{\pgfqpoint{0.066cm}{1.912cm}}{\pgfqpoint{0.04cm}{1.886cm}}
\pgfpathcurveto{\pgfqpoint{0.014cm}{1.86cm}}{\pgfqpoint{0cm}{1.825cm}}{\pgfqpoint{0cm}{1.789cm}}
\pgfpathcurveto{\pgfqpoint{0cm}{1.753cm}}{\pgfqpoint{0.014cm}{1.718cm}}{\pgfqpoint{0.04cm}{1.692cm}}
\pgfpathcurveto{\pgfqpoint{0.066cm}{1.667cm}}{\pgfqpoint{0.1cm}{1.652cm}}{\pgfqpoint{0.137cm}{1.652cm}}
\pgfpathcurveto{\pgfqpoint{0.173cm}{1.652cm}}{\pgfqpoint{0.207cm}{1.667cm}}{\pgfqpoint{0.233cm}{1.692cm}}
\pgfpathcurveto{\pgfqpoint{0.259cm}{1.718cm}}{\pgfqpoint{0.273cm}{1.753cm}}{\pgfqpoint{0.273cm}{1.789cm}}
\pgfusepath{fill}
\begin{pgfscope}
\pgfsetdash{}{0cm}
\pgfsetlinewidth{0.818mm}
\pgfsetmiterlimit{7.0}
\pgfpathmoveto{\pgfqpoint{0.682cm}{1.065cm}}
\pgfpathlineto{\pgfqpoint{0.679cm}{1.812cm}}
\pgfusepath{stroke}
\end{pgfscope}
\pgfpathmoveto{\pgfqpoint{0.815cm}{1.793cm}}
\pgfpathcurveto{\pgfqpoint{0.815cm}{1.829cm}}{\pgfqpoint{0.801cm}{1.864cm}}{\pgfqpoint{0.775cm}{1.89cm}}
\pgfpathcurveto{\pgfqpoint{0.75cm}{1.915cm}}{\pgfqpoint{0.715cm}{1.93cm}}{\pgfqpoint{0.679cm}{1.93cm}}
\pgfpathcurveto{\pgfqpoint{0.643cm}{1.93cm}}{\pgfqpoint{0.608cm}{1.915cm}}{\pgfqpoint{0.582cm}{1.89cm}}
\pgfpathcurveto{\pgfqpoint{0.557cm}{1.864cm}}{\pgfqpoint{0.542cm}{1.829cm}}{\pgfqpoint{0.542cm}{1.793cm}}
\pgfpathcurveto{\pgfqpoint{0.542cm}{1.756cm}}{\pgfqpoint{0.557cm}{1.722cm}}{\pgfqpoint{0.582cm}{1.696cm}}
\pgfpathcurveto{\pgfqpoint{0.608cm}{1.67cm}}{\pgfqpoint{0.643cm}{1.656cm}}{\pgfqpoint{0.679cm}{1.656cm}}
\pgfpathcurveto{\pgfqpoint{0.715cm}{1.656cm}}{\pgfqpoint{0.75cm}{1.67cm}}{\pgfqpoint{0.775cm}{1.696cm}}
\pgfpathcurveto{\pgfqpoint{0.801cm}{1.722cm}}{\pgfqpoint{0.815cm}{1.756cm}}{\pgfqpoint{0.815cm}{1.793cm}}
\pgfusepath{fill}
\pgfpathmoveto{\pgfqpoint{1.345cm}{1.765cm}}
\pgfpathcurveto{\pgfqpoint{1.345cm}{1.801cm}}{\pgfqpoint{1.331cm}{1.836cm}}{\pgfqpoint{1.305cm}{1.862cm}}
\pgfpathcurveto{\pgfqpoint{1.28cm}{1.887cm}}{\pgfqpoint{1.245cm}{1.902cm}}{\pgfqpoint{1.209cm}{1.902cm}}
\pgfpathcurveto{\pgfqpoint{1.172cm}{1.902cm}}{\pgfqpoint{1.138cm}{1.887cm}}{\pgfqpoint{1.112cm}{1.862cm}}
\pgfpathcurveto{\pgfqpoint{1.087cm}{1.836cm}}{\pgfqpoint{1.072cm}{1.801cm}}{\pgfqpoint{1.072cm}{1.765cm}}
\pgfpathcurveto{\pgfqpoint{1.072cm}{1.728cm}}{\pgfqpoint{1.087cm}{1.694cm}}{\pgfqpoint{1.112cm}{1.668cm}}
\pgfpathcurveto{\pgfqpoint{1.138cm}{1.642cm}}{\pgfqpoint{1.172cm}{1.628cm}}{\pgfqpoint{1.209cm}{1.628cm}}
\pgfpathcurveto{\pgfqpoint{1.245cm}{1.628cm}}{\pgfqpoint{1.28cm}{1.642cm}}{\pgfqpoint{1.305cm}{1.668cm}}
\pgfpathcurveto{\pgfqpoint{1.331cm}{1.694cm}}{\pgfqpoint{1.345cm}{1.728cm}}{\pgfqpoint{1.345cm}{1.765cm}}
\pgfusepath{fill}
\begin{pgfscope}
\pgfsetdash{}{0cm}
\pgfsetlinewidth{0.818mm}
\pgfsetroundcap
\pgfsetroundjoin
\pgfsetmiterlimit{7.0}
\pgfpathmoveto{\pgfqpoint{0.682cm}{1.065cm}}
\pgfpathlineto{\pgfqpoint{1.246cm}{0.315cm}}
\pgfpathlineto{\pgfqpoint{1.811cm}{1.065cm}}
\pgfusepath{stroke}
\end{pgfscope}
\pgfpathmoveto{\pgfqpoint{1.948cm}{1.065cm}}
\pgfpathcurveto{\pgfqpoint{1.948cm}{1.101cm}}{\pgfqpoint{1.933cm}{1.136cm}}{\pgfqpoint{1.907cm}{1.162cm}}
\pgfpathcurveto{\pgfqpoint{1.882cm}{1.187cm}}{\pgfqpoint{1.847cm}{1.202cm}}{\pgfqpoint{1.811cm}{1.202cm}}
\pgfpathcurveto{\pgfqpoint{1.775cm}{1.202cm}}{\pgfqpoint{1.74cm}{1.187cm}}{\pgfqpoint{1.714cm}{1.162cm}}
\pgfpathcurveto{\pgfqpoint{1.689cm}{1.136cm}}{\pgfqpoint{1.674cm}{1.101cm}}{\pgfqpoint{1.674cm}{1.065cm}}
\pgfpathcurveto{\pgfqpoint{1.674cm}{1.029cm}}{\pgfqpoint{1.689cm}{0.994cm}}{\pgfqpoint{1.714cm}{0.968cm}}
\pgfpathcurveto{\pgfqpoint{1.74cm}{0.942cm}}{\pgfqpoint{1.775cm}{0.928cm}}{\pgfqpoint{1.811cm}{0.928cm}}
\pgfpathcurveto{\pgfqpoint{1.847cm}{0.928cm}}{\pgfqpoint{1.882cm}{0.942cm}}{\pgfqpoint{1.907cm}{0.968cm}}
\pgfpathcurveto{\pgfqpoint{1.933cm}{0.994cm}}{\pgfqpoint{1.948cm}{1.029cm}}{\pgfqpoint{1.948cm}{1.065cm}}
\pgfusepath{fill}
\begin{pgfscope}
\pgfsetdash{}{0cm}
\pgfsetlinewidth{0.818mm}
\pgfsetmiterlimit{4.0}
\pgfpathmoveto{\pgfqpoint{1.383cm}{0.178cm}}
\pgfpathcurveto{\pgfqpoint{1.383cm}{0.214cm}}{\pgfqpoint{1.369cm}{0.249cm}}{\pgfqpoint{1.343cm}{0.275cm}}
\pgfpathcurveto{\pgfqpoint{1.317cm}{0.3cm}}{\pgfqpoint{1.283cm}{0.315cm}}{\pgfqpoint{1.246cm}{0.315cm}}
\pgfpathcurveto{\pgfqpoint{1.21cm}{0.315cm}}{\pgfqpoint{1.175cm}{0.3cm}}{\pgfqpoint{1.15cm}{0.275cm}}
\pgfpathcurveto{\pgfqpoint{1.124cm}{0.249cm}}{\pgfqpoint{1.11cm}{0.214cm}}{\pgfqpoint{1.11cm}{0.178cm}}
\pgfpathcurveto{\pgfqpoint{1.11cm}{0.141cm}}{\pgfqpoint{1.124cm}{0.107cm}}{\pgfqpoint{1.15cm}{0.081cm}}
\pgfpathcurveto{\pgfqpoint{1.175cm}{0.055cm}}{\pgfqpoint{1.21cm}{0.041cm}}{\pgfqpoint{1.246cm}{0.041cm}}
\pgfpathcurveto{\pgfqpoint{1.283cm}{0.041cm}}{\pgfqpoint{1.317cm}{0.055cm}}{\pgfqpoint{1.343cm}{0.081cm}}
\pgfpathcurveto{\pgfqpoint{1.369cm}{0.107cm}}{\pgfqpoint{1.383cm}{0.141cm}}{\pgfqpoint{1.383cm}{0.178cm}}
\pgfusepath{stroke}
\end{pgfscope}
\end{pgfscope}
\end{pgfscope}
\end{pgfscope}
\end{tikzpicture}}}}+\rmb{6(\phi+\psi)\prec\UU_{\leqslant}X^{\!\resizebox{!}{.8em}{
\begin{tikzpicture}
\pgfpathmoveto{\pgfqpoint{0cm}{-0.035cm}}
\pgfpathlineto{\pgfqpoint{1.976cm}{-0.035cm}}
\pgfpathlineto{\pgfqpoint{1.976cm}{1.94cm}}
\pgfpathlineto{\pgfqpoint{0cm}{1.94cm}}
\pgfpathclose
\pgfusepath{clip}
\begin{pgfscope}
\begin{pgfscope}
\pgfpathmoveto{\pgfqpoint{0cm}{-0.035cm}}
\pgfpathlineto{\pgfqpoint{1.976cm}{-0.035cm}}
\pgfpathlineto{\pgfqpoint{1.976cm}{1.94cm}}
\pgfpathlineto{\pgfqpoint{0cm}{1.94cm}}
\pgfpathclose
\pgfusepath{clip}
\begin{pgfscope}
\begin{pgfscope}
\pgfsetdash{}{0cm}
\pgfsetlinewidth{0.818mm}
\pgfsetroundcap
\pgfsetroundjoin
\pgfsetmiterlimit{7.0}
\definecolor{eps2pgf_color}{gray}{0}\pgfsetstrokecolor{eps2pgf_color}\pgfsetfillcolor{eps2pgf_color}
\pgfpathmoveto{\pgfqpoint{0.117cm}{1.815cm}}
\pgfpathlineto{\pgfqpoint{0.682cm}{1.065cm}}
\pgfpathlineto{\pgfqpoint{1.246cm}{1.815cm}}
\pgfusepath{stroke}
\end{pgfscope}
\definecolor{eps2pgf_color}{gray}{0}\pgfsetstrokecolor{eps2pgf_color}\pgfsetfillcolor{eps2pgf_color}
\pgfpathmoveto{\pgfqpoint{0.273cm}{1.789cm}}
\pgfpathcurveto{\pgfqpoint{0.273cm}{1.825cm}}{\pgfqpoint{0.259cm}{1.86cm}}{\pgfqpoint{0.233cm}{1.886cm}}
\pgfpathcurveto{\pgfqpoint{0.207cm}{1.912cm}}{\pgfqpoint{0.173cm}{1.926cm}}{\pgfqpoint{0.137cm}{1.926cm}}
\pgfpathcurveto{\pgfqpoint{0.1cm}{1.926cm}}{\pgfqpoint{0.066cm}{1.912cm}}{\pgfqpoint{0.04cm}{1.886cm}}
\pgfpathcurveto{\pgfqpoint{0.014cm}{1.86cm}}{\pgfqpoint{0cm}{1.825cm}}{\pgfqpoint{0cm}{1.789cm}}
\pgfpathcurveto{\pgfqpoint{0cm}{1.753cm}}{\pgfqpoint{0.014cm}{1.718cm}}{\pgfqpoint{0.04cm}{1.692cm}}
\pgfpathcurveto{\pgfqpoint{0.066cm}{1.667cm}}{\pgfqpoint{0.1cm}{1.652cm}}{\pgfqpoint{0.137cm}{1.652cm}}
\pgfpathcurveto{\pgfqpoint{0.173cm}{1.652cm}}{\pgfqpoint{0.207cm}{1.667cm}}{\pgfqpoint{0.233cm}{1.692cm}}
\pgfpathcurveto{\pgfqpoint{0.259cm}{1.718cm}}{\pgfqpoint{0.273cm}{1.753cm}}{\pgfqpoint{0.273cm}{1.789cm}}
\pgfusepath{fill}
\begin{pgfscope}
\pgfsetdash{}{0cm}
\pgfsetlinewidth{0.818mm}
\pgfsetmiterlimit{7.0}
\pgfpathmoveto{\pgfqpoint{0.682cm}{1.065cm}}
\pgfpathlineto{\pgfqpoint{0.679cm}{1.812cm}}
\pgfusepath{stroke}
\end{pgfscope}
\pgfpathmoveto{\pgfqpoint{0.815cm}{1.793cm}}
\pgfpathcurveto{\pgfqpoint{0.815cm}{1.829cm}}{\pgfqpoint{0.801cm}{1.864cm}}{\pgfqpoint{0.775cm}{1.89cm}}
\pgfpathcurveto{\pgfqpoint{0.75cm}{1.915cm}}{\pgfqpoint{0.715cm}{1.93cm}}{\pgfqpoint{0.679cm}{1.93cm}}
\pgfpathcurveto{\pgfqpoint{0.643cm}{1.93cm}}{\pgfqpoint{0.608cm}{1.915cm}}{\pgfqpoint{0.582cm}{1.89cm}}
\pgfpathcurveto{\pgfqpoint{0.557cm}{1.864cm}}{\pgfqpoint{0.542cm}{1.829cm}}{\pgfqpoint{0.542cm}{1.793cm}}
\pgfpathcurveto{\pgfqpoint{0.542cm}{1.756cm}}{\pgfqpoint{0.557cm}{1.722cm}}{\pgfqpoint{0.582cm}{1.696cm}}
\pgfpathcurveto{\pgfqpoint{0.608cm}{1.67cm}}{\pgfqpoint{0.643cm}{1.656cm}}{\pgfqpoint{0.679cm}{1.656cm}}
\pgfpathcurveto{\pgfqpoint{0.715cm}{1.656cm}}{\pgfqpoint{0.75cm}{1.67cm}}{\pgfqpoint{0.775cm}{1.696cm}}
\pgfpathcurveto{\pgfqpoint{0.801cm}{1.722cm}}{\pgfqpoint{0.815cm}{1.756cm}}{\pgfqpoint{0.815cm}{1.793cm}}
\pgfusepath{fill}
\pgfpathmoveto{\pgfqpoint{1.345cm}{1.765cm}}
\pgfpathcurveto{\pgfqpoint{1.345cm}{1.801cm}}{\pgfqpoint{1.331cm}{1.836cm}}{\pgfqpoint{1.305cm}{1.862cm}}
\pgfpathcurveto{\pgfqpoint{1.28cm}{1.887cm}}{\pgfqpoint{1.245cm}{1.902cm}}{\pgfqpoint{1.209cm}{1.902cm}}
\pgfpathcurveto{\pgfqpoint{1.172cm}{1.902cm}}{\pgfqpoint{1.138cm}{1.887cm}}{\pgfqpoint{1.112cm}{1.862cm}}
\pgfpathcurveto{\pgfqpoint{1.087cm}{1.836cm}}{\pgfqpoint{1.072cm}{1.801cm}}{\pgfqpoint{1.072cm}{1.765cm}}
\pgfpathcurveto{\pgfqpoint{1.072cm}{1.728cm}}{\pgfqpoint{1.087cm}{1.694cm}}{\pgfqpoint{1.112cm}{1.668cm}}
\pgfpathcurveto{\pgfqpoint{1.138cm}{1.642cm}}{\pgfqpoint{1.172cm}{1.628cm}}{\pgfqpoint{1.209cm}{1.628cm}}
\pgfpathcurveto{\pgfqpoint{1.245cm}{1.628cm}}{\pgfqpoint{1.28cm}{1.642cm}}{\pgfqpoint{1.305cm}{1.668cm}}
\pgfpathcurveto{\pgfqpoint{1.331cm}{1.694cm}}{\pgfqpoint{1.345cm}{1.728cm}}{\pgfqpoint{1.345cm}{1.765cm}}
\pgfusepath{fill}
\begin{pgfscope}
\pgfsetdash{}{0cm}
\pgfsetlinewidth{0.818mm}
\pgfsetroundcap
\pgfsetroundjoin
\pgfsetmiterlimit{7.0}
\pgfpathmoveto{\pgfqpoint{0.682cm}{1.065cm}}
\pgfpathlineto{\pgfqpoint{1.246cm}{0.315cm}}
\pgfpathlineto{\pgfqpoint{1.811cm}{1.065cm}}
\pgfusepath{stroke}
\end{pgfscope}
\pgfpathmoveto{\pgfqpoint{1.948cm}{1.065cm}}
\pgfpathcurveto{\pgfqpoint{1.948cm}{1.101cm}}{\pgfqpoint{1.933cm}{1.136cm}}{\pgfqpoint{1.907cm}{1.162cm}}
\pgfpathcurveto{\pgfqpoint{1.882cm}{1.187cm}}{\pgfqpoint{1.847cm}{1.202cm}}{\pgfqpoint{1.811cm}{1.202cm}}
\pgfpathcurveto{\pgfqpoint{1.775cm}{1.202cm}}{\pgfqpoint{1.74cm}{1.187cm}}{\pgfqpoint{1.714cm}{1.162cm}}
\pgfpathcurveto{\pgfqpoint{1.689cm}{1.136cm}}{\pgfqpoint{1.674cm}{1.101cm}}{\pgfqpoint{1.674cm}{1.065cm}}
\pgfpathcurveto{\pgfqpoint{1.674cm}{1.029cm}}{\pgfqpoint{1.689cm}{0.994cm}}{\pgfqpoint{1.714cm}{0.968cm}}
\pgfpathcurveto{\pgfqpoint{1.74cm}{0.942cm}}{\pgfqpoint{1.775cm}{0.928cm}}{\pgfqpoint{1.811cm}{0.928cm}}
\pgfpathcurveto{\pgfqpoint{1.847cm}{0.928cm}}{\pgfqpoint{1.882cm}{0.942cm}}{\pgfqpoint{1.907cm}{0.968cm}}
\pgfpathcurveto{\pgfqpoint{1.933cm}{0.994cm}}{\pgfqpoint{1.948cm}{1.029cm}}{\pgfqpoint{1.948cm}{1.065cm}}
\pgfusepath{fill}
\begin{pgfscope}
\pgfsetdash{}{0cm}
\pgfsetlinewidth{0.818mm}
\pgfsetmiterlimit{4.0}
\pgfpathmoveto{\pgfqpoint{1.383cm}{0.178cm}}
\pgfpathcurveto{\pgfqpoint{1.383cm}{0.214cm}}{\pgfqpoint{1.369cm}{0.249cm}}{\pgfqpoint{1.343cm}{0.275cm}}
\pgfpathcurveto{\pgfqpoint{1.317cm}{0.3cm}}{\pgfqpoint{1.283cm}{0.315cm}}{\pgfqpoint{1.246cm}{0.315cm}}
\pgfpathcurveto{\pgfqpoint{1.21cm}{0.315cm}}{\pgfqpoint{1.175cm}{0.3cm}}{\pgfqpoint{1.15cm}{0.275cm}}
\pgfpathcurveto{\pgfqpoint{1.124cm}{0.249cm}}{\pgfqpoint{1.11cm}{0.214cm}}{\pgfqpoint{1.11cm}{0.178cm}}
\pgfpathcurveto{\pgfqpoint{1.11cm}{0.141cm}}{\pgfqpoint{1.124cm}{0.107cm}}{\pgfqpoint{1.15cm}{0.081cm}}
\pgfpathcurveto{\pgfqpoint{1.175cm}{0.055cm}}{\pgfqpoint{1.21cm}{0.041cm}}{\pgfqpoint{1.246cm}{0.041cm}}
\pgfpathcurveto{\pgfqpoint{1.283cm}{0.041cm}}{\pgfqpoint{1.317cm}{0.055cm}}{\pgfqpoint{1.343cm}{0.081cm}}
\pgfpathcurveto{\pgfqpoint{1.369cm}{0.107cm}}{\pgfqpoint{1.383cm}{0.141cm}}{\pgfqpoint{1.383cm}{0.178cm}}
\pgfusepath{stroke}
\end{pgfscope}
\end{pgfscope}
\end{pgfscope}
\end{pgfscope}
\end{tikzpicture}}}}+\rmb{6(\phi+\psi)\succcurlyeqX^{\!\resizebox{!}{.8em}{
\begin{tikzpicture}
\pgfpathmoveto{\pgfqpoint{0cm}{-0.035cm}}
\pgfpathlineto{\pgfqpoint{1.976cm}{-0.035cm}}
\pgfpathlineto{\pgfqpoint{1.976cm}{1.94cm}}
\pgfpathlineto{\pgfqpoint{0cm}{1.94cm}}
\pgfpathclose
\pgfusepath{clip}
\begin{pgfscope}
\begin{pgfscope}
\pgfpathmoveto{\pgfqpoint{0cm}{-0.035cm}}
\pgfpathlineto{\pgfqpoint{1.976cm}{-0.035cm}}
\pgfpathlineto{\pgfqpoint{1.976cm}{1.94cm}}
\pgfpathlineto{\pgfqpoint{0cm}{1.94cm}}
\pgfpathclose
\pgfusepath{clip}
\begin{pgfscope}
\begin{pgfscope}
\pgfsetdash{}{0cm}
\pgfsetlinewidth{0.818mm}
\pgfsetroundcap
\pgfsetroundjoin
\pgfsetmiterlimit{7.0}
\definecolor{eps2pgf_color}{gray}{0}\pgfsetstrokecolor{eps2pgf_color}\pgfsetfillcolor{eps2pgf_color}
\pgfpathmoveto{\pgfqpoint{0.117cm}{1.815cm}}
\pgfpathlineto{\pgfqpoint{0.682cm}{1.065cm}}
\pgfpathlineto{\pgfqpoint{1.246cm}{1.815cm}}
\pgfusepath{stroke}
\end{pgfscope}
\definecolor{eps2pgf_color}{gray}{0}\pgfsetstrokecolor{eps2pgf_color}\pgfsetfillcolor{eps2pgf_color}
\pgfpathmoveto{\pgfqpoint{0.273cm}{1.789cm}}
\pgfpathcurveto{\pgfqpoint{0.273cm}{1.825cm}}{\pgfqpoint{0.259cm}{1.86cm}}{\pgfqpoint{0.233cm}{1.886cm}}
\pgfpathcurveto{\pgfqpoint{0.207cm}{1.912cm}}{\pgfqpoint{0.173cm}{1.926cm}}{\pgfqpoint{0.137cm}{1.926cm}}
\pgfpathcurveto{\pgfqpoint{0.1cm}{1.926cm}}{\pgfqpoint{0.066cm}{1.912cm}}{\pgfqpoint{0.04cm}{1.886cm}}
\pgfpathcurveto{\pgfqpoint{0.014cm}{1.86cm}}{\pgfqpoint{0cm}{1.825cm}}{\pgfqpoint{0cm}{1.789cm}}
\pgfpathcurveto{\pgfqpoint{0cm}{1.753cm}}{\pgfqpoint{0.014cm}{1.718cm}}{\pgfqpoint{0.04cm}{1.692cm}}
\pgfpathcurveto{\pgfqpoint{0.066cm}{1.667cm}}{\pgfqpoint{0.1cm}{1.652cm}}{\pgfqpoint{0.137cm}{1.652cm}}
\pgfpathcurveto{\pgfqpoint{0.173cm}{1.652cm}}{\pgfqpoint{0.207cm}{1.667cm}}{\pgfqpoint{0.233cm}{1.692cm}}
\pgfpathcurveto{\pgfqpoint{0.259cm}{1.718cm}}{\pgfqpoint{0.273cm}{1.753cm}}{\pgfqpoint{0.273cm}{1.789cm}}
\pgfusepath{fill}
\begin{pgfscope}
\pgfsetdash{}{0cm}
\pgfsetlinewidth{0.818mm}
\pgfsetmiterlimit{7.0}
\pgfpathmoveto{\pgfqpoint{0.682cm}{1.065cm}}
\pgfpathlineto{\pgfqpoint{0.679cm}{1.812cm}}
\pgfusepath{stroke}
\end{pgfscope}
\pgfpathmoveto{\pgfqpoint{0.815cm}{1.793cm}}
\pgfpathcurveto{\pgfqpoint{0.815cm}{1.829cm}}{\pgfqpoint{0.801cm}{1.864cm}}{\pgfqpoint{0.775cm}{1.89cm}}
\pgfpathcurveto{\pgfqpoint{0.75cm}{1.915cm}}{\pgfqpoint{0.715cm}{1.93cm}}{\pgfqpoint{0.679cm}{1.93cm}}
\pgfpathcurveto{\pgfqpoint{0.643cm}{1.93cm}}{\pgfqpoint{0.608cm}{1.915cm}}{\pgfqpoint{0.582cm}{1.89cm}}
\pgfpathcurveto{\pgfqpoint{0.557cm}{1.864cm}}{\pgfqpoint{0.542cm}{1.829cm}}{\pgfqpoint{0.542cm}{1.793cm}}
\pgfpathcurveto{\pgfqpoint{0.542cm}{1.756cm}}{\pgfqpoint{0.557cm}{1.722cm}}{\pgfqpoint{0.582cm}{1.696cm}}
\pgfpathcurveto{\pgfqpoint{0.608cm}{1.67cm}}{\pgfqpoint{0.643cm}{1.656cm}}{\pgfqpoint{0.679cm}{1.656cm}}
\pgfpathcurveto{\pgfqpoint{0.715cm}{1.656cm}}{\pgfqpoint{0.75cm}{1.67cm}}{\pgfqpoint{0.775cm}{1.696cm}}
\pgfpathcurveto{\pgfqpoint{0.801cm}{1.722cm}}{\pgfqpoint{0.815cm}{1.756cm}}{\pgfqpoint{0.815cm}{1.793cm}}
\pgfusepath{fill}
\pgfpathmoveto{\pgfqpoint{1.345cm}{1.765cm}}
\pgfpathcurveto{\pgfqpoint{1.345cm}{1.801cm}}{\pgfqpoint{1.331cm}{1.836cm}}{\pgfqpoint{1.305cm}{1.862cm}}
\pgfpathcurveto{\pgfqpoint{1.28cm}{1.887cm}}{\pgfqpoint{1.245cm}{1.902cm}}{\pgfqpoint{1.209cm}{1.902cm}}
\pgfpathcurveto{\pgfqpoint{1.172cm}{1.902cm}}{\pgfqpoint{1.138cm}{1.887cm}}{\pgfqpoint{1.112cm}{1.862cm}}
\pgfpathcurveto{\pgfqpoint{1.087cm}{1.836cm}}{\pgfqpoint{1.072cm}{1.801cm}}{\pgfqpoint{1.072cm}{1.765cm}}
\pgfpathcurveto{\pgfqpoint{1.072cm}{1.728cm}}{\pgfqpoint{1.087cm}{1.694cm}}{\pgfqpoint{1.112cm}{1.668cm}}
\pgfpathcurveto{\pgfqpoint{1.138cm}{1.642cm}}{\pgfqpoint{1.172cm}{1.628cm}}{\pgfqpoint{1.209cm}{1.628cm}}
\pgfpathcurveto{\pgfqpoint{1.245cm}{1.628cm}}{\pgfqpoint{1.28cm}{1.642cm}}{\pgfqpoint{1.305cm}{1.668cm}}
\pgfpathcurveto{\pgfqpoint{1.331cm}{1.694cm}}{\pgfqpoint{1.345cm}{1.728cm}}{\pgfqpoint{1.345cm}{1.765cm}}
\pgfusepath{fill}
\begin{pgfscope}
\pgfsetdash{}{0cm}
\pgfsetlinewidth{0.818mm}
\pgfsetroundcap
\pgfsetroundjoin
\pgfsetmiterlimit{7.0}
\pgfpathmoveto{\pgfqpoint{0.682cm}{1.065cm}}
\pgfpathlineto{\pgfqpoint{1.246cm}{0.315cm}}
\pgfpathlineto{\pgfqpoint{1.811cm}{1.065cm}}
\pgfusepath{stroke}
\end{pgfscope}
\pgfpathmoveto{\pgfqpoint{1.948cm}{1.065cm}}
\pgfpathcurveto{\pgfqpoint{1.948cm}{1.101cm}}{\pgfqpoint{1.933cm}{1.136cm}}{\pgfqpoint{1.907cm}{1.162cm}}
\pgfpathcurveto{\pgfqpoint{1.882cm}{1.187cm}}{\pgfqpoint{1.847cm}{1.202cm}}{\pgfqpoint{1.811cm}{1.202cm}}
\pgfpathcurveto{\pgfqpoint{1.775cm}{1.202cm}}{\pgfqpoint{1.74cm}{1.187cm}}{\pgfqpoint{1.714cm}{1.162cm}}
\pgfpathcurveto{\pgfqpoint{1.689cm}{1.136cm}}{\pgfqpoint{1.674cm}{1.101cm}}{\pgfqpoint{1.674cm}{1.065cm}}
\pgfpathcurveto{\pgfqpoint{1.674cm}{1.029cm}}{\pgfqpoint{1.689cm}{0.994cm}}{\pgfqpoint{1.714cm}{0.968cm}}
\pgfpathcurveto{\pgfqpoint{1.74cm}{0.942cm}}{\pgfqpoint{1.775cm}{0.928cm}}{\pgfqpoint{1.811cm}{0.928cm}}
\pgfpathcurveto{\pgfqpoint{1.847cm}{0.928cm}}{\pgfqpoint{1.882cm}{0.942cm}}{\pgfqpoint{1.907cm}{0.968cm}}
\pgfpathcurveto{\pgfqpoint{1.933cm}{0.994cm}}{\pgfqpoint{1.948cm}{1.029cm}}{\pgfqpoint{1.948cm}{1.065cm}}
\pgfusepath{fill}
\begin{pgfscope}
\pgfsetdash{}{0cm}
\pgfsetlinewidth{0.818mm}
\pgfsetmiterlimit{4.0}
\pgfpathmoveto{\pgfqpoint{1.383cm}{0.178cm}}
\pgfpathcurveto{\pgfqpoint{1.383cm}{0.214cm}}{\pgfqpoint{1.369cm}{0.249cm}}{\pgfqpoint{1.343cm}{0.275cm}}
\pgfpathcurveto{\pgfqpoint{1.317cm}{0.3cm}}{\pgfqpoint{1.283cm}{0.315cm}}{\pgfqpoint{1.246cm}{0.315cm}}
\pgfpathcurveto{\pgfqpoint{1.21cm}{0.315cm}}{\pgfqpoint{1.175cm}{0.3cm}}{\pgfqpoint{1.15cm}{0.275cm}}
\pgfpathcurveto{\pgfqpoint{1.124cm}{0.249cm}}{\pgfqpoint{1.11cm}{0.214cm}}{\pgfqpoint{1.11cm}{0.178cm}}
\pgfpathcurveto{\pgfqpoint{1.11cm}{0.141cm}}{\pgfqpoint{1.124cm}{0.107cm}}{\pgfqpoint{1.15cm}{0.081cm}}
\pgfpathcurveto{\pgfqpoint{1.175cm}{0.055cm}}{\pgfqpoint{1.21cm}{0.041cm}}{\pgfqpoint{1.246cm}{0.041cm}}
\pgfpathcurveto{\pgfqpoint{1.283cm}{0.041cm}}{\pgfqpoint{1.317cm}{0.055cm}}{\pgfqpoint{1.343cm}{0.081cm}}
\pgfpathcurveto{\pgfqpoint{1.369cm}{0.107cm}}{\pgfqpoint{1.383cm}{0.141cm}}{\pgfqpoint{1.383cm}{0.178cm}}
\pgfusepath{stroke}
\end{pgfscope}
\end{pgfscope}
\end{pgfscope}
\end{pgfscope}
\end{tikzpicture}}}},
\end{align*}
\begin{align}\label{eq:91}
\rmg{3X(\phi+\psi)^2}=\rmm{3\UU_> X\succ(\phi+\psi)^2}+\rmb{3\UU_\leqslant X\succ (\phi+\psi)^2}+\rmb{3 X\preccurlyeq(\phi+\psi)^2}.
\end{align}

As mentioned above,  the concrete choice of the localizers $\UU_\leq$, $\UU_{>}$ in the above changes from line to line.  In particular, it will be seen below that the localization of $X$ in \eqref{eq:92}, \eqref{eq:93} is different from \eqref{eq:91}. The precise choice of these parameters will be made in Section \ref{ssec:phi1} below.

Now, let $\rmm{\Phi}$ be the sum of all the  magenta terms above and $\rmb{\Psi}$ the sum of all the blue terms. \rmbb{More precisely,} 
\begin{align*}
\rmm{\Phi}&:=-\rmm{3\llbracket X^2 \rrbracket\succX^{\!\resizebox{0.6em}{!}{
\begin{tikzpicture}
\pgfpathmoveto{\pgfqpoint{0cm}{-0.035cm}}
\pgfpathlineto{\pgfqpoint{1.376cm}{-0.035cm}}
\pgfpathlineto{\pgfqpoint{1.376cm}{1.552cm}}
\pgfpathlineto{\pgfqpoint{0cm}{1.552cm}}
\pgfpathclose
\pgfusepath{clip}
\begin{pgfscope}
\begin{pgfscope}
\pgfpathmoveto{\pgfqpoint{0cm}{-0.035cm}}
\pgfpathlineto{\pgfqpoint{1.376cm}{-0.035cm}}
\pgfpathlineto{\pgfqpoint{1.376cm}{1.552cm}}
\pgfpathlineto{\pgfqpoint{0cm}{1.552cm}}
\pgfpathclose
\pgfusepath{clip}
\begin{pgfscope}
\begin{pgfscope}
\pgfsetdash{}{0cm}
\pgfsetlinewidth{0.818mm}
\pgfsetroundcap
\pgfsetroundjoin
\pgfsetmiterlimit{7.0}
\definecolor{eps2pgf_color}{gray}{0}\pgfsetstrokecolor{eps2pgf_color}\pgfsetfillcolor{eps2pgf_color}
\pgfpathmoveto{\pgfqpoint{0.117cm}{1.421cm}}
\pgfpathlineto{\pgfqpoint{0.682cm}{0.671cm}}
\pgfpathlineto{\pgfqpoint{1.246cm}{1.421cm}}
\pgfusepath{stroke}
\end{pgfscope}
\definecolor{eps2pgf_color}{gray}{0}\pgfsetstrokecolor{eps2pgf_color}\pgfsetfillcolor{eps2pgf_color}
\pgfpathmoveto{\pgfqpoint{0.273cm}{1.395cm}}
\pgfpathcurveto{\pgfqpoint{0.273cm}{1.432cm}}{\pgfqpoint{0.259cm}{1.467cm}}{\pgfqpoint{0.233cm}{1.492cm}}
\pgfpathcurveto{\pgfqpoint{0.207cm}{1.518cm}}{\pgfqpoint{0.173cm}{1.532cm}}{\pgfqpoint{0.137cm}{1.532cm}}
\pgfpathcurveto{\pgfqpoint{0.1cm}{1.532cm}}{\pgfqpoint{0.066cm}{1.518cm}}{\pgfqpoint{0.04cm}{1.492cm}}
\pgfpathcurveto{\pgfqpoint{0.014cm}{1.467cm}}{\pgfqpoint{0cm}{1.432cm}}{\pgfqpoint{0cm}{1.395cm}}
\pgfpathcurveto{\pgfqpoint{0cm}{1.359cm}}{\pgfqpoint{0.014cm}{1.324cm}}{\pgfqpoint{0.04cm}{1.299cm}}
\pgfpathcurveto{\pgfqpoint{0.066cm}{1.273cm}}{\pgfqpoint{0.1cm}{1.258cm}}{\pgfqpoint{0.137cm}{1.258cm}}
\pgfpathcurveto{\pgfqpoint{0.173cm}{1.258cm}}{\pgfqpoint{0.207cm}{1.273cm}}{\pgfqpoint{0.233cm}{1.299cm}}
\pgfpathcurveto{\pgfqpoint{0.259cm}{1.324cm}}{\pgfqpoint{0.273cm}{1.359cm}}{\pgfqpoint{0.273cm}{1.395cm}}
\pgfusepath{fill}
\begin{pgfscope}
\pgfsetdash{}{0cm}
\pgfsetlinewidth{0.818mm}
\pgfsetmiterlimit{7.0}
\pgfpathmoveto{\pgfqpoint{0.682cm}{0.671cm}}
\pgfpathlineto{\pgfqpoint{0.679cm}{1.418cm}}
\pgfusepath{stroke}
\end{pgfscope}
\pgfpathmoveto{\pgfqpoint{0.815cm}{1.399cm}}
\pgfpathcurveto{\pgfqpoint{0.815cm}{1.435cm}}{\pgfqpoint{0.801cm}{1.47cm}}{\pgfqpoint{0.775cm}{1.496cm}}
\pgfpathcurveto{\pgfqpoint{0.75cm}{1.521cm}}{\pgfqpoint{0.715cm}{1.536cm}}{\pgfqpoint{0.679cm}{1.536cm}}
\pgfpathcurveto{\pgfqpoint{0.643cm}{1.536cm}}{\pgfqpoint{0.608cm}{1.521cm}}{\pgfqpoint{0.582cm}{1.496cm}}
\pgfpathcurveto{\pgfqpoint{0.557cm}{1.47cm}}{\pgfqpoint{0.542cm}{1.435cm}}{\pgfqpoint{0.542cm}{1.399cm}}
\pgfpathcurveto{\pgfqpoint{0.542cm}{1.363cm}}{\pgfqpoint{0.557cm}{1.328cm}}{\pgfqpoint{0.582cm}{1.302cm}}
\pgfpathcurveto{\pgfqpoint{0.608cm}{1.276cm}}{\pgfqpoint{0.643cm}{1.262cm}}{\pgfqpoint{0.679cm}{1.262cm}}
\pgfpathcurveto{\pgfqpoint{0.715cm}{1.262cm}}{\pgfqpoint{0.75cm}{1.276cm}}{\pgfqpoint{0.775cm}{1.302cm}}
\pgfpathcurveto{\pgfqpoint{0.801cm}{1.328cm}}{\pgfqpoint{0.815cm}{1.363cm}}{\pgfqpoint{0.815cm}{1.399cm}}
\pgfusepath{fill}
\pgfpathmoveto{\pgfqpoint{1.345cm}{1.371cm}}
\pgfpathcurveto{\pgfqpoint{1.345cm}{1.408cm}}{\pgfqpoint{1.331cm}{1.442cm}}{\pgfqpoint{1.305cm}{1.468cm}}
\pgfpathcurveto{\pgfqpoint{1.28cm}{1.494cm}}{\pgfqpoint{1.245cm}{1.508cm}}{\pgfqpoint{1.209cm}{1.508cm}}
\pgfpathcurveto{\pgfqpoint{1.172cm}{1.508cm}}{\pgfqpoint{1.138cm}{1.494cm}}{\pgfqpoint{1.112cm}{1.468cm}}
\pgfpathcurveto{\pgfqpoint{1.087cm}{1.442cm}}{\pgfqpoint{1.072cm}{1.408cm}}{\pgfqpoint{1.072cm}{1.371cm}}
\pgfpathcurveto{\pgfqpoint{1.072cm}{1.335cm}}{\pgfqpoint{1.087cm}{1.3cm}}{\pgfqpoint{1.112cm}{1.274cm}}
\pgfpathcurveto{\pgfqpoint{1.138cm}{1.249cm}}{\pgfqpoint{1.172cm}{1.234cm}}{\pgfqpoint{1.209cm}{1.234cm}}
\pgfpathcurveto{\pgfqpoint{1.245cm}{1.234cm}}{\pgfqpoint{1.28cm}{1.249cm}}{\pgfqpoint{1.305cm}{1.274cm}}
\pgfpathcurveto{\pgfqpoint{1.331cm}{1.3cm}}{\pgfqpoint{1.345cm}{1.335cm}}{\pgfqpoint{1.345cm}{1.371cm}}
\pgfusepath{fill}
\begin{pgfscope}
\pgfsetdash{}{0cm}
\pgfsetlinewidth{0.818mm}
\pgfsetroundcap
\pgfsetmiterlimit{4.0}
\pgfpathmoveto{\pgfqpoint{0.682cm}{0.671cm}}
\pgfpathlineto{\pgfqpoint{0.682cm}{0.042cm}}
\pgfusepath{stroke}
\end{pgfscope}
\end{pgfscope}
\end{pgfscope}
\end{pgfscope}
\end{tikzpicture}}}}-\rmm{3\llbracket X^2 \rrbracket\precX^{\!\resizebox{0.6em}{!}{
\begin{tikzpicture}
\pgfpathmoveto{\pgfqpoint{0cm}{-0.035cm}}
\pgfpathlineto{\pgfqpoint{1.376cm}{-0.035cm}}
\pgfpathlineto{\pgfqpoint{1.376cm}{1.552cm}}
\pgfpathlineto{\pgfqpoint{0cm}{1.552cm}}
\pgfpathclose
\pgfusepath{clip}
\begin{pgfscope}
\begin{pgfscope}
\pgfpathmoveto{\pgfqpoint{0cm}{-0.035cm}}
\pgfpathlineto{\pgfqpoint{1.376cm}{-0.035cm}}
\pgfpathlineto{\pgfqpoint{1.376cm}{1.552cm}}
\pgfpathlineto{\pgfqpoint{0cm}{1.552cm}}
\pgfpathclose
\pgfusepath{clip}
\begin{pgfscope}
\begin{pgfscope}
\pgfsetdash{}{0cm}
\pgfsetlinewidth{0.818mm}
\pgfsetroundcap
\pgfsetroundjoin
\pgfsetmiterlimit{7.0}
\definecolor{eps2pgf_color}{gray}{0}\pgfsetstrokecolor{eps2pgf_color}\pgfsetfillcolor{eps2pgf_color}
\pgfpathmoveto{\pgfqpoint{0.117cm}{1.421cm}}
\pgfpathlineto{\pgfqpoint{0.682cm}{0.671cm}}
\pgfpathlineto{\pgfqpoint{1.246cm}{1.421cm}}
\pgfusepath{stroke}
\end{pgfscope}
\definecolor{eps2pgf_color}{gray}{0}\pgfsetstrokecolor{eps2pgf_color}\pgfsetfillcolor{eps2pgf_color}
\pgfpathmoveto{\pgfqpoint{0.273cm}{1.395cm}}
\pgfpathcurveto{\pgfqpoint{0.273cm}{1.432cm}}{\pgfqpoint{0.259cm}{1.467cm}}{\pgfqpoint{0.233cm}{1.492cm}}
\pgfpathcurveto{\pgfqpoint{0.207cm}{1.518cm}}{\pgfqpoint{0.173cm}{1.532cm}}{\pgfqpoint{0.137cm}{1.532cm}}
\pgfpathcurveto{\pgfqpoint{0.1cm}{1.532cm}}{\pgfqpoint{0.066cm}{1.518cm}}{\pgfqpoint{0.04cm}{1.492cm}}
\pgfpathcurveto{\pgfqpoint{0.014cm}{1.467cm}}{\pgfqpoint{0cm}{1.432cm}}{\pgfqpoint{0cm}{1.395cm}}
\pgfpathcurveto{\pgfqpoint{0cm}{1.359cm}}{\pgfqpoint{0.014cm}{1.324cm}}{\pgfqpoint{0.04cm}{1.299cm}}
\pgfpathcurveto{\pgfqpoint{0.066cm}{1.273cm}}{\pgfqpoint{0.1cm}{1.258cm}}{\pgfqpoint{0.137cm}{1.258cm}}
\pgfpathcurveto{\pgfqpoint{0.173cm}{1.258cm}}{\pgfqpoint{0.207cm}{1.273cm}}{\pgfqpoint{0.233cm}{1.299cm}}
\pgfpathcurveto{\pgfqpoint{0.259cm}{1.324cm}}{\pgfqpoint{0.273cm}{1.359cm}}{\pgfqpoint{0.273cm}{1.395cm}}
\pgfusepath{fill}
\begin{pgfscope}
\pgfsetdash{}{0cm}
\pgfsetlinewidth{0.818mm}
\pgfsetmiterlimit{7.0}
\pgfpathmoveto{\pgfqpoint{0.682cm}{0.671cm}}
\pgfpathlineto{\pgfqpoint{0.679cm}{1.418cm}}
\pgfusepath{stroke}
\end{pgfscope}
\pgfpathmoveto{\pgfqpoint{0.815cm}{1.399cm}}
\pgfpathcurveto{\pgfqpoint{0.815cm}{1.435cm}}{\pgfqpoint{0.801cm}{1.47cm}}{\pgfqpoint{0.775cm}{1.496cm}}
\pgfpathcurveto{\pgfqpoint{0.75cm}{1.521cm}}{\pgfqpoint{0.715cm}{1.536cm}}{\pgfqpoint{0.679cm}{1.536cm}}
\pgfpathcurveto{\pgfqpoint{0.643cm}{1.536cm}}{\pgfqpoint{0.608cm}{1.521cm}}{\pgfqpoint{0.582cm}{1.496cm}}
\pgfpathcurveto{\pgfqpoint{0.557cm}{1.47cm}}{\pgfqpoint{0.542cm}{1.435cm}}{\pgfqpoint{0.542cm}{1.399cm}}
\pgfpathcurveto{\pgfqpoint{0.542cm}{1.363cm}}{\pgfqpoint{0.557cm}{1.328cm}}{\pgfqpoint{0.582cm}{1.302cm}}
\pgfpathcurveto{\pgfqpoint{0.608cm}{1.276cm}}{\pgfqpoint{0.643cm}{1.262cm}}{\pgfqpoint{0.679cm}{1.262cm}}
\pgfpathcurveto{\pgfqpoint{0.715cm}{1.262cm}}{\pgfqpoint{0.75cm}{1.276cm}}{\pgfqpoint{0.775cm}{1.302cm}}
\pgfpathcurveto{\pgfqpoint{0.801cm}{1.328cm}}{\pgfqpoint{0.815cm}{1.363cm}}{\pgfqpoint{0.815cm}{1.399cm}}
\pgfusepath{fill}
\pgfpathmoveto{\pgfqpoint{1.345cm}{1.371cm}}
\pgfpathcurveto{\pgfqpoint{1.345cm}{1.408cm}}{\pgfqpoint{1.331cm}{1.442cm}}{\pgfqpoint{1.305cm}{1.468cm}}
\pgfpathcurveto{\pgfqpoint{1.28cm}{1.494cm}}{\pgfqpoint{1.245cm}{1.508cm}}{\pgfqpoint{1.209cm}{1.508cm}}
\pgfpathcurveto{\pgfqpoint{1.172cm}{1.508cm}}{\pgfqpoint{1.138cm}{1.494cm}}{\pgfqpoint{1.112cm}{1.468cm}}
\pgfpathcurveto{\pgfqpoint{1.087cm}{1.442cm}}{\pgfqpoint{1.072cm}{1.408cm}}{\pgfqpoint{1.072cm}{1.371cm}}
\pgfpathcurveto{\pgfqpoint{1.072cm}{1.335cm}}{\pgfqpoint{1.087cm}{1.3cm}}{\pgfqpoint{1.112cm}{1.274cm}}
\pgfpathcurveto{\pgfqpoint{1.138cm}{1.249cm}}{\pgfqpoint{1.172cm}{1.234cm}}{\pgfqpoint{1.209cm}{1.234cm}}
\pgfpathcurveto{\pgfqpoint{1.245cm}{1.234cm}}{\pgfqpoint{1.28cm}{1.249cm}}{\pgfqpoint{1.305cm}{1.274cm}}
\pgfpathcurveto{\pgfqpoint{1.331cm}{1.3cm}}{\pgfqpoint{1.345cm}{1.335cm}}{\pgfqpoint{1.345cm}{1.371cm}}
\pgfusepath{fill}
\begin{pgfscope}
\pgfsetdash{}{0cm}
\pgfsetlinewidth{0.818mm}
\pgfsetroundcap
\pgfsetmiterlimit{4.0}
\pgfpathmoveto{\pgfqpoint{0.682cm}{0.671cm}}
\pgfpathlineto{\pgfqpoint{0.682cm}{0.042cm}}
\pgfusepath{stroke}
\end{pgfscope}
\end{pgfscope}
\end{pgfscope}
\end{pgfscope}
\end{tikzpicture}}}}- \rmm{3X^{\!\resizebox{!}{.8em}{
\begin{tikzpicture}
\pgfpathmoveto{\pgfqpoint{0cm}{-0.035cm}}
\pgfpathlineto{\pgfqpoint{1.976cm}{-0.035cm}}
\pgfpathlineto{\pgfqpoint{1.976cm}{1.94cm}}
\pgfpathlineto{\pgfqpoint{0cm}{1.94cm}}
\pgfpathclose
\pgfusepath{clip}
\begin{pgfscope}
\begin{pgfscope}
\pgfpathmoveto{\pgfqpoint{0cm}{-0.035cm}}
\pgfpathlineto{\pgfqpoint{1.976cm}{-0.035cm}}
\pgfpathlineto{\pgfqpoint{1.976cm}{1.94cm}}
\pgfpathlineto{\pgfqpoint{0cm}{1.94cm}}
\pgfpathclose
\pgfusepath{clip}
\begin{pgfscope}
\begin{pgfscope}
\pgfsetdash{}{0cm}
\pgfsetlinewidth{0.818mm}
\pgfsetroundcap
\pgfsetroundjoin
\pgfsetmiterlimit{7.0}
\definecolor{eps2pgf_color}{gray}{0}\pgfsetstrokecolor{eps2pgf_color}\pgfsetfillcolor{eps2pgf_color}
\pgfpathmoveto{\pgfqpoint{0.117cm}{1.815cm}}
\pgfpathlineto{\pgfqpoint{0.682cm}{1.065cm}}
\pgfpathlineto{\pgfqpoint{1.246cm}{1.815cm}}
\pgfusepath{stroke}
\end{pgfscope}
\definecolor{eps2pgf_color}{gray}{0}\pgfsetstrokecolor{eps2pgf_color}\pgfsetfillcolor{eps2pgf_color}
\pgfpathmoveto{\pgfqpoint{0.273cm}{1.789cm}}
\pgfpathcurveto{\pgfqpoint{0.273cm}{1.825cm}}{\pgfqpoint{0.259cm}{1.86cm}}{\pgfqpoint{0.233cm}{1.886cm}}
\pgfpathcurveto{\pgfqpoint{0.207cm}{1.912cm}}{\pgfqpoint{0.173cm}{1.926cm}}{\pgfqpoint{0.137cm}{1.926cm}}
\pgfpathcurveto{\pgfqpoint{0.1cm}{1.926cm}}{\pgfqpoint{0.066cm}{1.912cm}}{\pgfqpoint{0.04cm}{1.886cm}}
\pgfpathcurveto{\pgfqpoint{0.014cm}{1.86cm}}{\pgfqpoint{0cm}{1.825cm}}{\pgfqpoint{0cm}{1.789cm}}
\pgfpathcurveto{\pgfqpoint{0cm}{1.753cm}}{\pgfqpoint{0.014cm}{1.718cm}}{\pgfqpoint{0.04cm}{1.692cm}}
\pgfpathcurveto{\pgfqpoint{0.066cm}{1.667cm}}{\pgfqpoint{0.1cm}{1.652cm}}{\pgfqpoint{0.137cm}{1.652cm}}
\pgfpathcurveto{\pgfqpoint{0.173cm}{1.652cm}}{\pgfqpoint{0.207cm}{1.667cm}}{\pgfqpoint{0.233cm}{1.692cm}}
\pgfpathcurveto{\pgfqpoint{0.259cm}{1.718cm}}{\pgfqpoint{0.273cm}{1.753cm}}{\pgfqpoint{0.273cm}{1.789cm}}
\pgfusepath{fill}
\begin{pgfscope}
\pgfsetdash{}{0cm}
\pgfsetlinewidth{0.818mm}
\pgfsetmiterlimit{7.0}
\pgfpathmoveto{\pgfqpoint{0.682cm}{1.065cm}}
\pgfpathlineto{\pgfqpoint{0.679cm}{1.812cm}}
\pgfusepath{stroke}
\end{pgfscope}
\pgfpathmoveto{\pgfqpoint{0.815cm}{1.793cm}}
\pgfpathcurveto{\pgfqpoint{0.815cm}{1.829cm}}{\pgfqpoint{0.801cm}{1.864cm}}{\pgfqpoint{0.775cm}{1.89cm}}
\pgfpathcurveto{\pgfqpoint{0.75cm}{1.915cm}}{\pgfqpoint{0.715cm}{1.93cm}}{\pgfqpoint{0.679cm}{1.93cm}}
\pgfpathcurveto{\pgfqpoint{0.643cm}{1.93cm}}{\pgfqpoint{0.608cm}{1.915cm}}{\pgfqpoint{0.582cm}{1.89cm}}
\pgfpathcurveto{\pgfqpoint{0.557cm}{1.864cm}}{\pgfqpoint{0.542cm}{1.829cm}}{\pgfqpoint{0.542cm}{1.793cm}}
\pgfpathcurveto{\pgfqpoint{0.542cm}{1.756cm}}{\pgfqpoint{0.557cm}{1.722cm}}{\pgfqpoint{0.582cm}{1.696cm}}
\pgfpathcurveto{\pgfqpoint{0.608cm}{1.67cm}}{\pgfqpoint{0.643cm}{1.656cm}}{\pgfqpoint{0.679cm}{1.656cm}}
\pgfpathcurveto{\pgfqpoint{0.715cm}{1.656cm}}{\pgfqpoint{0.75cm}{1.67cm}}{\pgfqpoint{0.775cm}{1.696cm}}
\pgfpathcurveto{\pgfqpoint{0.801cm}{1.722cm}}{\pgfqpoint{0.815cm}{1.756cm}}{\pgfqpoint{0.815cm}{1.793cm}}
\pgfusepath{fill}
\pgfpathmoveto{\pgfqpoint{1.345cm}{1.765cm}}
\pgfpathcurveto{\pgfqpoint{1.345cm}{1.801cm}}{\pgfqpoint{1.331cm}{1.836cm}}{\pgfqpoint{1.305cm}{1.862cm}}
\pgfpathcurveto{\pgfqpoint{1.28cm}{1.887cm}}{\pgfqpoint{1.245cm}{1.902cm}}{\pgfqpoint{1.209cm}{1.902cm}}
\pgfpathcurveto{\pgfqpoint{1.172cm}{1.902cm}}{\pgfqpoint{1.138cm}{1.887cm}}{\pgfqpoint{1.112cm}{1.862cm}}
\pgfpathcurveto{\pgfqpoint{1.087cm}{1.836cm}}{\pgfqpoint{1.072cm}{1.801cm}}{\pgfqpoint{1.072cm}{1.765cm}}
\pgfpathcurveto{\pgfqpoint{1.072cm}{1.728cm}}{\pgfqpoint{1.087cm}{1.694cm}}{\pgfqpoint{1.112cm}{1.668cm}}
\pgfpathcurveto{\pgfqpoint{1.138cm}{1.642cm}}{\pgfqpoint{1.172cm}{1.628cm}}{\pgfqpoint{1.209cm}{1.628cm}}
\pgfpathcurveto{\pgfqpoint{1.245cm}{1.628cm}}{\pgfqpoint{1.28cm}{1.642cm}}{\pgfqpoint{1.305cm}{1.668cm}}
\pgfpathcurveto{\pgfqpoint{1.331cm}{1.694cm}}{\pgfqpoint{1.345cm}{1.728cm}}{\pgfqpoint{1.345cm}{1.765cm}}
\pgfusepath{fill}
\begin{pgfscope}
\pgfsetdash{}{0cm}
\pgfsetlinewidth{0.818mm}
\pgfsetroundcap
\pgfsetroundjoin
\pgfsetmiterlimit{7.0}
\pgfpathmoveto{\pgfqpoint{0.682cm}{1.065cm}}
\pgfpathlineto{\pgfqpoint{1.246cm}{0.315cm}}
\pgfpathlineto{\pgfqpoint{1.811cm}{1.065cm}}
\pgfusepath{stroke}
\end{pgfscope}
\pgfpathmoveto{\pgfqpoint{1.948cm}{1.065cm}}
\pgfpathcurveto{\pgfqpoint{1.948cm}{1.101cm}}{\pgfqpoint{1.933cm}{1.136cm}}{\pgfqpoint{1.907cm}{1.162cm}}
\pgfpathcurveto{\pgfqpoint{1.882cm}{1.187cm}}{\pgfqpoint{1.847cm}{1.202cm}}{\pgfqpoint{1.811cm}{1.202cm}}
\pgfpathcurveto{\pgfqpoint{1.775cm}{1.202cm}}{\pgfqpoint{1.74cm}{1.187cm}}{\pgfqpoint{1.714cm}{1.162cm}}
\pgfpathcurveto{\pgfqpoint{1.689cm}{1.136cm}}{\pgfqpoint{1.674cm}{1.101cm}}{\pgfqpoint{1.674cm}{1.065cm}}
\pgfpathcurveto{\pgfqpoint{1.674cm}{1.029cm}}{\pgfqpoint{1.689cm}{0.994cm}}{\pgfqpoint{1.714cm}{0.968cm}}
\pgfpathcurveto{\pgfqpoint{1.74cm}{0.942cm}}{\pgfqpoint{1.775cm}{0.928cm}}{\pgfqpoint{1.811cm}{0.928cm}}
\pgfpathcurveto{\pgfqpoint{1.847cm}{0.928cm}}{\pgfqpoint{1.882cm}{0.942cm}}{\pgfqpoint{1.907cm}{0.968cm}}
\pgfpathcurveto{\pgfqpoint{1.933cm}{0.994cm}}{\pgfqpoint{1.948cm}{1.029cm}}{\pgfqpoint{1.948cm}{1.065cm}}
\pgfusepath{fill}
\begin{pgfscope}
\pgfsetdash{}{0cm}
\pgfsetlinewidth{0.818mm}
\pgfsetmiterlimit{7.0}
\pgfpathmoveto{\pgfqpoint{1.246cm}{0.315cm}}
\pgfpathlineto{\pgfqpoint{1.244cm}{1.061cm}}
\pgfusepath{stroke}
\end{pgfscope}
\pgfpathmoveto{\pgfqpoint{1.38cm}{1.065cm}}
\pgfpathcurveto{\pgfqpoint{1.38cm}{1.101cm}}{\pgfqpoint{1.366cm}{1.136cm}}{\pgfqpoint{1.34cm}{1.162cm}}
\pgfpathcurveto{\pgfqpoint{1.315cm}{1.187cm}}{\pgfqpoint{1.28cm}{1.202cm}}{\pgfqpoint{1.244cm}{1.202cm}}
\pgfpathcurveto{\pgfqpoint{1.207cm}{1.202cm}}{\pgfqpoint{1.173cm}{1.187cm}}{\pgfqpoint{1.147cm}{1.162cm}}
\pgfpathcurveto{\pgfqpoint{1.121cm}{1.136cm}}{\pgfqpoint{1.107cm}{1.101cm}}{\pgfqpoint{1.107cm}{1.065cm}}
\pgfpathcurveto{\pgfqpoint{1.107cm}{1.029cm}}{\pgfqpoint{1.121cm}{0.994cm}}{\pgfqpoint{1.147cm}{0.968cm}}
\pgfpathcurveto{\pgfqpoint{1.173cm}{0.942cm}}{\pgfqpoint{1.207cm}{0.928cm}}{\pgfqpoint{1.244cm}{0.928cm}}
\pgfpathcurveto{\pgfqpoint{1.28cm}{0.928cm}}{\pgfqpoint{1.315cm}{0.942cm}}{\pgfqpoint{1.34cm}{0.968cm}}
\pgfpathcurveto{\pgfqpoint{1.366cm}{0.994cm}}{\pgfqpoint{1.38cm}{1.029cm}}{\pgfqpoint{1.38cm}{1.065cm}}
\pgfusepath{fill}
\begin{pgfscope}
\pgfsetdash{}{0cm}
\pgfsetlinewidth{0.818mm}
\pgfsetmiterlimit{4.0}
\pgfpathmoveto{\pgfqpoint{1.383cm}{0.178cm}}
\pgfpathcurveto{\pgfqpoint{1.383cm}{0.214cm}}{\pgfqpoint{1.369cm}{0.249cm}}{\pgfqpoint{1.343cm}{0.275cm}}
\pgfpathcurveto{\pgfqpoint{1.317cm}{0.3cm}}{\pgfqpoint{1.283cm}{0.315cm}}{\pgfqpoint{1.246cm}{0.315cm}}
\pgfpathcurveto{\pgfqpoint{1.21cm}{0.315cm}}{\pgfqpoint{1.175cm}{0.3cm}}{\pgfqpoint{1.15cm}{0.275cm}}
\pgfpathcurveto{\pgfqpoint{1.124cm}{0.249cm}}{\pgfqpoint{1.11cm}{0.214cm}}{\pgfqpoint{1.11cm}{0.178cm}}
\pgfpathcurveto{\pgfqpoint{1.11cm}{0.141cm}}{\pgfqpoint{1.124cm}{0.107cm}}{\pgfqpoint{1.15cm}{0.081cm}}
\pgfpathcurveto{\pgfqpoint{1.175cm}{0.055cm}}{\pgfqpoint{1.21cm}{0.041cm}}{\pgfqpoint{1.246cm}{0.041cm}}
\pgfpathcurveto{\pgfqpoint{1.283cm}{0.041cm}}{\pgfqpoint{1.317cm}{0.055cm}}{\pgfqpoint{1.343cm}{0.081cm}}
\pgfpathcurveto{\pgfqpoint{1.369cm}{0.107cm}}{\pgfqpoint{1.383cm}{0.141cm}}{\pgfqpoint{1.383cm}{0.178cm}}
\pgfusepath{stroke}
\end{pgfscope}
\end{pgfscope}
\end{pgfscope}
\end{pgfscope}
\end{tikzpicture}}}}+\rmm{3X\succ(X^{\!\resizebox{0.6em}{!}{
\begin{tikzpicture}
\pgfpathmoveto{\pgfqpoint{0cm}{-0.035cm}}
\pgfpathlineto{\pgfqpoint{1.376cm}{-0.035cm}}
\pgfpathlineto{\pgfqpoint{1.376cm}{1.552cm}}
\pgfpathlineto{\pgfqpoint{0cm}{1.552cm}}
\pgfpathclose
\pgfusepath{clip}
\begin{pgfscope}
\begin{pgfscope}
\pgfpathmoveto{\pgfqpoint{0cm}{-0.035cm}}
\pgfpathlineto{\pgfqpoint{1.376cm}{-0.035cm}}
\pgfpathlineto{\pgfqpoint{1.376cm}{1.552cm}}
\pgfpathlineto{\pgfqpoint{0cm}{1.552cm}}
\pgfpathclose
\pgfusepath{clip}
\begin{pgfscope}
\begin{pgfscope}
\pgfsetdash{}{0cm}
\pgfsetlinewidth{0.818mm}
\pgfsetroundcap
\pgfsetroundjoin
\pgfsetmiterlimit{7.0}
\definecolor{eps2pgf_color}{gray}{0}\pgfsetstrokecolor{eps2pgf_color}\pgfsetfillcolor{eps2pgf_color}
\pgfpathmoveto{\pgfqpoint{0.117cm}{1.421cm}}
\pgfpathlineto{\pgfqpoint{0.682cm}{0.671cm}}
\pgfpathlineto{\pgfqpoint{1.246cm}{1.421cm}}
\pgfusepath{stroke}
\end{pgfscope}
\definecolor{eps2pgf_color}{gray}{0}\pgfsetstrokecolor{eps2pgf_color}\pgfsetfillcolor{eps2pgf_color}
\pgfpathmoveto{\pgfqpoint{0.273cm}{1.395cm}}
\pgfpathcurveto{\pgfqpoint{0.273cm}{1.432cm}}{\pgfqpoint{0.259cm}{1.467cm}}{\pgfqpoint{0.233cm}{1.492cm}}
\pgfpathcurveto{\pgfqpoint{0.207cm}{1.518cm}}{\pgfqpoint{0.173cm}{1.532cm}}{\pgfqpoint{0.137cm}{1.532cm}}
\pgfpathcurveto{\pgfqpoint{0.1cm}{1.532cm}}{\pgfqpoint{0.066cm}{1.518cm}}{\pgfqpoint{0.04cm}{1.492cm}}
\pgfpathcurveto{\pgfqpoint{0.014cm}{1.467cm}}{\pgfqpoint{0cm}{1.432cm}}{\pgfqpoint{0cm}{1.395cm}}
\pgfpathcurveto{\pgfqpoint{0cm}{1.359cm}}{\pgfqpoint{0.014cm}{1.324cm}}{\pgfqpoint{0.04cm}{1.299cm}}
\pgfpathcurveto{\pgfqpoint{0.066cm}{1.273cm}}{\pgfqpoint{0.1cm}{1.258cm}}{\pgfqpoint{0.137cm}{1.258cm}}
\pgfpathcurveto{\pgfqpoint{0.173cm}{1.258cm}}{\pgfqpoint{0.207cm}{1.273cm}}{\pgfqpoint{0.233cm}{1.299cm}}
\pgfpathcurveto{\pgfqpoint{0.259cm}{1.324cm}}{\pgfqpoint{0.273cm}{1.359cm}}{\pgfqpoint{0.273cm}{1.395cm}}
\pgfusepath{fill}
\begin{pgfscope}
\pgfsetdash{}{0cm}
\pgfsetlinewidth{0.818mm}
\pgfsetmiterlimit{7.0}
\pgfpathmoveto{\pgfqpoint{0.682cm}{0.671cm}}
\pgfpathlineto{\pgfqpoint{0.679cm}{1.418cm}}
\pgfusepath{stroke}
\end{pgfscope}
\pgfpathmoveto{\pgfqpoint{0.815cm}{1.399cm}}
\pgfpathcurveto{\pgfqpoint{0.815cm}{1.435cm}}{\pgfqpoint{0.801cm}{1.47cm}}{\pgfqpoint{0.775cm}{1.496cm}}
\pgfpathcurveto{\pgfqpoint{0.75cm}{1.521cm}}{\pgfqpoint{0.715cm}{1.536cm}}{\pgfqpoint{0.679cm}{1.536cm}}
\pgfpathcurveto{\pgfqpoint{0.643cm}{1.536cm}}{\pgfqpoint{0.608cm}{1.521cm}}{\pgfqpoint{0.582cm}{1.496cm}}
\pgfpathcurveto{\pgfqpoint{0.557cm}{1.47cm}}{\pgfqpoint{0.542cm}{1.435cm}}{\pgfqpoint{0.542cm}{1.399cm}}
\pgfpathcurveto{\pgfqpoint{0.542cm}{1.363cm}}{\pgfqpoint{0.557cm}{1.328cm}}{\pgfqpoint{0.582cm}{1.302cm}}
\pgfpathcurveto{\pgfqpoint{0.608cm}{1.276cm}}{\pgfqpoint{0.643cm}{1.262cm}}{\pgfqpoint{0.679cm}{1.262cm}}
\pgfpathcurveto{\pgfqpoint{0.715cm}{1.262cm}}{\pgfqpoint{0.75cm}{1.276cm}}{\pgfqpoint{0.775cm}{1.302cm}}
\pgfpathcurveto{\pgfqpoint{0.801cm}{1.328cm}}{\pgfqpoint{0.815cm}{1.363cm}}{\pgfqpoint{0.815cm}{1.399cm}}
\pgfusepath{fill}
\pgfpathmoveto{\pgfqpoint{1.345cm}{1.371cm}}
\pgfpathcurveto{\pgfqpoint{1.345cm}{1.408cm}}{\pgfqpoint{1.331cm}{1.442cm}}{\pgfqpoint{1.305cm}{1.468cm}}
\pgfpathcurveto{\pgfqpoint{1.28cm}{1.494cm}}{\pgfqpoint{1.245cm}{1.508cm}}{\pgfqpoint{1.209cm}{1.508cm}}
\pgfpathcurveto{\pgfqpoint{1.172cm}{1.508cm}}{\pgfqpoint{1.138cm}{1.494cm}}{\pgfqpoint{1.112cm}{1.468cm}}
\pgfpathcurveto{\pgfqpoint{1.087cm}{1.442cm}}{\pgfqpoint{1.072cm}{1.408cm}}{\pgfqpoint{1.072cm}{1.371cm}}
\pgfpathcurveto{\pgfqpoint{1.072cm}{1.335cm}}{\pgfqpoint{1.087cm}{1.3cm}}{\pgfqpoint{1.112cm}{1.274cm}}
\pgfpathcurveto{\pgfqpoint{1.138cm}{1.249cm}}{\pgfqpoint{1.172cm}{1.234cm}}{\pgfqpoint{1.209cm}{1.234cm}}
\pgfpathcurveto{\pgfqpoint{1.245cm}{1.234cm}}{\pgfqpoint{1.28cm}{1.249cm}}{\pgfqpoint{1.305cm}{1.274cm}}
\pgfpathcurveto{\pgfqpoint{1.331cm}{1.3cm}}{\pgfqpoint{1.345cm}{1.335cm}}{\pgfqpoint{1.345cm}{1.371cm}}
\pgfusepath{fill}
\begin{pgfscope}
\pgfsetdash{}{0cm}
\pgfsetlinewidth{0.818mm}
\pgfsetroundcap
\pgfsetmiterlimit{4.0}
\pgfpathmoveto{\pgfqpoint{0.682cm}{0.671cm}}
\pgfpathlineto{\pgfqpoint{0.682cm}{0.042cm}}
\pgfusepath{stroke}
\end{pgfscope}
\end{pgfscope}
\end{pgfscope}
\end{pgfscope}
\end{tikzpicture}}})^2}+\rmm{3X\prec(X^{\!\resizebox{0.6em}{!}{
\begin{tikzpicture}
\pgfpathmoveto{\pgfqpoint{0cm}{-0.035cm}}
\pgfpathlineto{\pgfqpoint{1.376cm}{-0.035cm}}
\pgfpathlineto{\pgfqpoint{1.376cm}{1.552cm}}
\pgfpathlineto{\pgfqpoint{0cm}{1.552cm}}
\pgfpathclose
\pgfusepath{clip}
\begin{pgfscope}
\begin{pgfscope}
\pgfpathmoveto{\pgfqpoint{0cm}{-0.035cm}}
\pgfpathlineto{\pgfqpoint{1.376cm}{-0.035cm}}
\pgfpathlineto{\pgfqpoint{1.376cm}{1.552cm}}
\pgfpathlineto{\pgfqpoint{0cm}{1.552cm}}
\pgfpathclose
\pgfusepath{clip}
\begin{pgfscope}
\begin{pgfscope}
\pgfsetdash{}{0cm}
\pgfsetlinewidth{0.818mm}
\pgfsetroundcap
\pgfsetroundjoin
\pgfsetmiterlimit{7.0}
\definecolor{eps2pgf_color}{gray}{0}\pgfsetstrokecolor{eps2pgf_color}\pgfsetfillcolor{eps2pgf_color}
\pgfpathmoveto{\pgfqpoint{0.117cm}{1.421cm}}
\pgfpathlineto{\pgfqpoint{0.682cm}{0.671cm}}
\pgfpathlineto{\pgfqpoint{1.246cm}{1.421cm}}
\pgfusepath{stroke}
\end{pgfscope}
\definecolor{eps2pgf_color}{gray}{0}\pgfsetstrokecolor{eps2pgf_color}\pgfsetfillcolor{eps2pgf_color}
\pgfpathmoveto{\pgfqpoint{0.273cm}{1.395cm}}
\pgfpathcurveto{\pgfqpoint{0.273cm}{1.432cm}}{\pgfqpoint{0.259cm}{1.467cm}}{\pgfqpoint{0.233cm}{1.492cm}}
\pgfpathcurveto{\pgfqpoint{0.207cm}{1.518cm}}{\pgfqpoint{0.173cm}{1.532cm}}{\pgfqpoint{0.137cm}{1.532cm}}
\pgfpathcurveto{\pgfqpoint{0.1cm}{1.532cm}}{\pgfqpoint{0.066cm}{1.518cm}}{\pgfqpoint{0.04cm}{1.492cm}}
\pgfpathcurveto{\pgfqpoint{0.014cm}{1.467cm}}{\pgfqpoint{0cm}{1.432cm}}{\pgfqpoint{0cm}{1.395cm}}
\pgfpathcurveto{\pgfqpoint{0cm}{1.359cm}}{\pgfqpoint{0.014cm}{1.324cm}}{\pgfqpoint{0.04cm}{1.299cm}}
\pgfpathcurveto{\pgfqpoint{0.066cm}{1.273cm}}{\pgfqpoint{0.1cm}{1.258cm}}{\pgfqpoint{0.137cm}{1.258cm}}
\pgfpathcurveto{\pgfqpoint{0.173cm}{1.258cm}}{\pgfqpoint{0.207cm}{1.273cm}}{\pgfqpoint{0.233cm}{1.299cm}}
\pgfpathcurveto{\pgfqpoint{0.259cm}{1.324cm}}{\pgfqpoint{0.273cm}{1.359cm}}{\pgfqpoint{0.273cm}{1.395cm}}
\pgfusepath{fill}
\begin{pgfscope}
\pgfsetdash{}{0cm}
\pgfsetlinewidth{0.818mm}
\pgfsetmiterlimit{7.0}
\pgfpathmoveto{\pgfqpoint{0.682cm}{0.671cm}}
\pgfpathlineto{\pgfqpoint{0.679cm}{1.418cm}}
\pgfusepath{stroke}
\end{pgfscope}
\pgfpathmoveto{\pgfqpoint{0.815cm}{1.399cm}}
\pgfpathcurveto{\pgfqpoint{0.815cm}{1.435cm}}{\pgfqpoint{0.801cm}{1.47cm}}{\pgfqpoint{0.775cm}{1.496cm}}
\pgfpathcurveto{\pgfqpoint{0.75cm}{1.521cm}}{\pgfqpoint{0.715cm}{1.536cm}}{\pgfqpoint{0.679cm}{1.536cm}}
\pgfpathcurveto{\pgfqpoint{0.643cm}{1.536cm}}{\pgfqpoint{0.608cm}{1.521cm}}{\pgfqpoint{0.582cm}{1.496cm}}
\pgfpathcurveto{\pgfqpoint{0.557cm}{1.47cm}}{\pgfqpoint{0.542cm}{1.435cm}}{\pgfqpoint{0.542cm}{1.399cm}}
\pgfpathcurveto{\pgfqpoint{0.542cm}{1.363cm}}{\pgfqpoint{0.557cm}{1.328cm}}{\pgfqpoint{0.582cm}{1.302cm}}
\pgfpathcurveto{\pgfqpoint{0.608cm}{1.276cm}}{\pgfqpoint{0.643cm}{1.262cm}}{\pgfqpoint{0.679cm}{1.262cm}}
\pgfpathcurveto{\pgfqpoint{0.715cm}{1.262cm}}{\pgfqpoint{0.75cm}{1.276cm}}{\pgfqpoint{0.775cm}{1.302cm}}
\pgfpathcurveto{\pgfqpoint{0.801cm}{1.328cm}}{\pgfqpoint{0.815cm}{1.363cm}}{\pgfqpoint{0.815cm}{1.399cm}}
\pgfusepath{fill}
\pgfpathmoveto{\pgfqpoint{1.345cm}{1.371cm}}
\pgfpathcurveto{\pgfqpoint{1.345cm}{1.408cm}}{\pgfqpoint{1.331cm}{1.442cm}}{\pgfqpoint{1.305cm}{1.468cm}}
\pgfpathcurveto{\pgfqpoint{1.28cm}{1.494cm}}{\pgfqpoint{1.245cm}{1.508cm}}{\pgfqpoint{1.209cm}{1.508cm}}
\pgfpathcurveto{\pgfqpoint{1.172cm}{1.508cm}}{\pgfqpoint{1.138cm}{1.494cm}}{\pgfqpoint{1.112cm}{1.468cm}}
\pgfpathcurveto{\pgfqpoint{1.087cm}{1.442cm}}{\pgfqpoint{1.072cm}{1.408cm}}{\pgfqpoint{1.072cm}{1.371cm}}
\pgfpathcurveto{\pgfqpoint{1.072cm}{1.335cm}}{\pgfqpoint{1.087cm}{1.3cm}}{\pgfqpoint{1.112cm}{1.274cm}}
\pgfpathcurveto{\pgfqpoint{1.138cm}{1.249cm}}{\pgfqpoint{1.172cm}{1.234cm}}{\pgfqpoint{1.209cm}{1.234cm}}
\pgfpathcurveto{\pgfqpoint{1.245cm}{1.234cm}}{\pgfqpoint{1.28cm}{1.249cm}}{\pgfqpoint{1.305cm}{1.274cm}}
\pgfpathcurveto{\pgfqpoint{1.331cm}{1.3cm}}{\pgfqpoint{1.345cm}{1.335cm}}{\pgfqpoint{1.345cm}{1.371cm}}
\pgfusepath{fill}
\begin{pgfscope}
\pgfsetdash{}{0cm}
\pgfsetlinewidth{0.818mm}
\pgfsetroundcap
\pgfsetmiterlimit{4.0}
\pgfpathmoveto{\pgfqpoint{0.682cm}{0.671cm}}
\pgfpathlineto{\pgfqpoint{0.682cm}{0.042cm}}
\pgfusepath{stroke}
\end{pgfscope}
\end{pgfscope}
\end{pgfscope}
\end{pgfscope}
\end{tikzpicture}}})^2}+\rmm{6X^{\!\resizebox{0.6em}{!}{
\begin{tikzpicture}
\pgfpathmoveto{\pgfqpoint{0cm}{-0.035cm}}
\pgfpathlineto{\pgfqpoint{1.376cm}{-0.035cm}}
\pgfpathlineto{\pgfqpoint{1.376cm}{1.552cm}}
\pgfpathlineto{\pgfqpoint{0cm}{1.552cm}}
\pgfpathclose
\pgfusepath{clip}
\begin{pgfscope}
\begin{pgfscope}
\pgfpathmoveto{\pgfqpoint{0cm}{-0.035cm}}
\pgfpathlineto{\pgfqpoint{1.376cm}{-0.035cm}}
\pgfpathlineto{\pgfqpoint{1.376cm}{1.552cm}}
\pgfpathlineto{\pgfqpoint{0cm}{1.552cm}}
\pgfpathclose
\pgfusepath{clip}
\begin{pgfscope}
\begin{pgfscope}
\pgfsetdash{}{0cm}
\pgfsetlinewidth{0.818mm}
\pgfsetroundcap
\pgfsetroundjoin
\pgfsetmiterlimit{7.0}
\definecolor{eps2pgf_color}{gray}{0}\pgfsetstrokecolor{eps2pgf_color}\pgfsetfillcolor{eps2pgf_color}
\pgfpathmoveto{\pgfqpoint{0.117cm}{1.421cm}}
\pgfpathlineto{\pgfqpoint{0.682cm}{0.671cm}}
\pgfpathlineto{\pgfqpoint{1.246cm}{1.421cm}}
\pgfusepath{stroke}
\end{pgfscope}
\definecolor{eps2pgf_color}{gray}{0}\pgfsetstrokecolor{eps2pgf_color}\pgfsetfillcolor{eps2pgf_color}
\pgfpathmoveto{\pgfqpoint{0.273cm}{1.395cm}}
\pgfpathcurveto{\pgfqpoint{0.273cm}{1.432cm}}{\pgfqpoint{0.259cm}{1.467cm}}{\pgfqpoint{0.233cm}{1.492cm}}
\pgfpathcurveto{\pgfqpoint{0.207cm}{1.518cm}}{\pgfqpoint{0.173cm}{1.532cm}}{\pgfqpoint{0.137cm}{1.532cm}}
\pgfpathcurveto{\pgfqpoint{0.1cm}{1.532cm}}{\pgfqpoint{0.066cm}{1.518cm}}{\pgfqpoint{0.04cm}{1.492cm}}
\pgfpathcurveto{\pgfqpoint{0.014cm}{1.467cm}}{\pgfqpoint{0cm}{1.432cm}}{\pgfqpoint{0cm}{1.395cm}}
\pgfpathcurveto{\pgfqpoint{0cm}{1.359cm}}{\pgfqpoint{0.014cm}{1.324cm}}{\pgfqpoint{0.04cm}{1.299cm}}
\pgfpathcurveto{\pgfqpoint{0.066cm}{1.273cm}}{\pgfqpoint{0.1cm}{1.258cm}}{\pgfqpoint{0.137cm}{1.258cm}}
\pgfpathcurveto{\pgfqpoint{0.173cm}{1.258cm}}{\pgfqpoint{0.207cm}{1.273cm}}{\pgfqpoint{0.233cm}{1.299cm}}
\pgfpathcurveto{\pgfqpoint{0.259cm}{1.324cm}}{\pgfqpoint{0.273cm}{1.359cm}}{\pgfqpoint{0.273cm}{1.395cm}}
\pgfusepath{fill}
\begin{pgfscope}
\pgfsetdash{}{0cm}
\pgfsetlinewidth{0.818mm}
\pgfsetmiterlimit{7.0}
\pgfpathmoveto{\pgfqpoint{0.682cm}{0.671cm}}
\pgfpathlineto{\pgfqpoint{0.679cm}{1.418cm}}
\pgfusepath{stroke}
\end{pgfscope}
\pgfpathmoveto{\pgfqpoint{0.815cm}{1.399cm}}
\pgfpathcurveto{\pgfqpoint{0.815cm}{1.435cm}}{\pgfqpoint{0.801cm}{1.47cm}}{\pgfqpoint{0.775cm}{1.496cm}}
\pgfpathcurveto{\pgfqpoint{0.75cm}{1.521cm}}{\pgfqpoint{0.715cm}{1.536cm}}{\pgfqpoint{0.679cm}{1.536cm}}
\pgfpathcurveto{\pgfqpoint{0.643cm}{1.536cm}}{\pgfqpoint{0.608cm}{1.521cm}}{\pgfqpoint{0.582cm}{1.496cm}}
\pgfpathcurveto{\pgfqpoint{0.557cm}{1.47cm}}{\pgfqpoint{0.542cm}{1.435cm}}{\pgfqpoint{0.542cm}{1.399cm}}
\pgfpathcurveto{\pgfqpoint{0.542cm}{1.363cm}}{\pgfqpoint{0.557cm}{1.328cm}}{\pgfqpoint{0.582cm}{1.302cm}}
\pgfpathcurveto{\pgfqpoint{0.608cm}{1.276cm}}{\pgfqpoint{0.643cm}{1.262cm}}{\pgfqpoint{0.679cm}{1.262cm}}
\pgfpathcurveto{\pgfqpoint{0.715cm}{1.262cm}}{\pgfqpoint{0.75cm}{1.276cm}}{\pgfqpoint{0.775cm}{1.302cm}}
\pgfpathcurveto{\pgfqpoint{0.801cm}{1.328cm}}{\pgfqpoint{0.815cm}{1.363cm}}{\pgfqpoint{0.815cm}{1.399cm}}
\pgfusepath{fill}
\pgfpathmoveto{\pgfqpoint{1.345cm}{1.371cm}}
\pgfpathcurveto{\pgfqpoint{1.345cm}{1.408cm}}{\pgfqpoint{1.331cm}{1.442cm}}{\pgfqpoint{1.305cm}{1.468cm}}
\pgfpathcurveto{\pgfqpoint{1.28cm}{1.494cm}}{\pgfqpoint{1.245cm}{1.508cm}}{\pgfqpoint{1.209cm}{1.508cm}}
\pgfpathcurveto{\pgfqpoint{1.172cm}{1.508cm}}{\pgfqpoint{1.138cm}{1.494cm}}{\pgfqpoint{1.112cm}{1.468cm}}
\pgfpathcurveto{\pgfqpoint{1.087cm}{1.442cm}}{\pgfqpoint{1.072cm}{1.408cm}}{\pgfqpoint{1.072cm}{1.371cm}}
\pgfpathcurveto{\pgfqpoint{1.072cm}{1.335cm}}{\pgfqpoint{1.087cm}{1.3cm}}{\pgfqpoint{1.112cm}{1.274cm}}
\pgfpathcurveto{\pgfqpoint{1.138cm}{1.249cm}}{\pgfqpoint{1.172cm}{1.234cm}}{\pgfqpoint{1.209cm}{1.234cm}}
\pgfpathcurveto{\pgfqpoint{1.245cm}{1.234cm}}{\pgfqpoint{1.28cm}{1.249cm}}{\pgfqpoint{1.305cm}{1.274cm}}
\pgfpathcurveto{\pgfqpoint{1.331cm}{1.3cm}}{\pgfqpoint{1.345cm}{1.335cm}}{\pgfqpoint{1.345cm}{1.371cm}}
\pgfusepath{fill}
\begin{pgfscope}
\pgfsetdash{}{0cm}
\pgfsetlinewidth{0.818mm}
\pgfsetroundcap
\pgfsetmiterlimit{4.0}
\pgfpathmoveto{\pgfqpoint{0.682cm}{0.671cm}}
\pgfpathlineto{\pgfqpoint{0.682cm}{0.042cm}}
\pgfusepath{stroke}
\end{pgfscope}
\end{pgfscope}
\end{pgfscope}
\end{pgfscope}
\end{tikzpicture}}}X^{\!\resizebox{!}{.8em}{
\begin{tikzpicture}
\pgfpathmoveto{\pgfqpoint{0cm}{-0.035cm}}
\pgfpathlineto{\pgfqpoint{1.976cm}{-0.035cm}}
\pgfpathlineto{\pgfqpoint{1.976cm}{1.94cm}}
\pgfpathlineto{\pgfqpoint{0cm}{1.94cm}}
\pgfpathclose
\pgfusepath{clip}
\begin{pgfscope}
\begin{pgfscope}
\pgfpathmoveto{\pgfqpoint{0cm}{-0.035cm}}
\pgfpathlineto{\pgfqpoint{1.976cm}{-0.035cm}}
\pgfpathlineto{\pgfqpoint{1.976cm}{1.94cm}}
\pgfpathlineto{\pgfqpoint{0cm}{1.94cm}}
\pgfpathclose
\pgfusepath{clip}
\begin{pgfscope}
\begin{pgfscope}
\pgfsetdash{}{0cm}
\pgfsetlinewidth{0.818mm}
\pgfsetroundcap
\pgfsetroundjoin
\pgfsetmiterlimit{7.0}
\definecolor{eps2pgf_color}{gray}{0}\pgfsetstrokecolor{eps2pgf_color}\pgfsetfillcolor{eps2pgf_color}
\pgfpathmoveto{\pgfqpoint{0.117cm}{1.815cm}}
\pgfpathlineto{\pgfqpoint{0.682cm}{1.065cm}}
\pgfpathlineto{\pgfqpoint{1.246cm}{1.815cm}}
\pgfusepath{stroke}
\end{pgfscope}
\definecolor{eps2pgf_color}{gray}{0}\pgfsetstrokecolor{eps2pgf_color}\pgfsetfillcolor{eps2pgf_color}
\pgfpathmoveto{\pgfqpoint{0.273cm}{1.789cm}}
\pgfpathcurveto{\pgfqpoint{0.273cm}{1.825cm}}{\pgfqpoint{0.259cm}{1.86cm}}{\pgfqpoint{0.233cm}{1.886cm}}
\pgfpathcurveto{\pgfqpoint{0.207cm}{1.912cm}}{\pgfqpoint{0.173cm}{1.926cm}}{\pgfqpoint{0.137cm}{1.926cm}}
\pgfpathcurveto{\pgfqpoint{0.1cm}{1.926cm}}{\pgfqpoint{0.066cm}{1.912cm}}{\pgfqpoint{0.04cm}{1.886cm}}
\pgfpathcurveto{\pgfqpoint{0.014cm}{1.86cm}}{\pgfqpoint{0cm}{1.825cm}}{\pgfqpoint{0cm}{1.789cm}}
\pgfpathcurveto{\pgfqpoint{0cm}{1.753cm}}{\pgfqpoint{0.014cm}{1.718cm}}{\pgfqpoint{0.04cm}{1.692cm}}
\pgfpathcurveto{\pgfqpoint{0.066cm}{1.667cm}}{\pgfqpoint{0.1cm}{1.652cm}}{\pgfqpoint{0.137cm}{1.652cm}}
\pgfpathcurveto{\pgfqpoint{0.173cm}{1.652cm}}{\pgfqpoint{0.207cm}{1.667cm}}{\pgfqpoint{0.233cm}{1.692cm}}
\pgfpathcurveto{\pgfqpoint{0.259cm}{1.718cm}}{\pgfqpoint{0.273cm}{1.753cm}}{\pgfqpoint{0.273cm}{1.789cm}}
\pgfusepath{fill}
\begin{pgfscope}
\pgfsetdash{}{0cm}
\pgfsetlinewidth{0.818mm}
\pgfsetmiterlimit{7.0}
\pgfpathmoveto{\pgfqpoint{0.682cm}{1.065cm}}
\pgfpathlineto{\pgfqpoint{0.679cm}{1.812cm}}
\pgfusepath{stroke}
\end{pgfscope}
\pgfpathmoveto{\pgfqpoint{0.815cm}{1.793cm}}
\pgfpathcurveto{\pgfqpoint{0.815cm}{1.829cm}}{\pgfqpoint{0.801cm}{1.864cm}}{\pgfqpoint{0.775cm}{1.89cm}}
\pgfpathcurveto{\pgfqpoint{0.75cm}{1.915cm}}{\pgfqpoint{0.715cm}{1.93cm}}{\pgfqpoint{0.679cm}{1.93cm}}
\pgfpathcurveto{\pgfqpoint{0.643cm}{1.93cm}}{\pgfqpoint{0.608cm}{1.915cm}}{\pgfqpoint{0.582cm}{1.89cm}}
\pgfpathcurveto{\pgfqpoint{0.557cm}{1.864cm}}{\pgfqpoint{0.542cm}{1.829cm}}{\pgfqpoint{0.542cm}{1.793cm}}
\pgfpathcurveto{\pgfqpoint{0.542cm}{1.756cm}}{\pgfqpoint{0.557cm}{1.722cm}}{\pgfqpoint{0.582cm}{1.696cm}}
\pgfpathcurveto{\pgfqpoint{0.608cm}{1.67cm}}{\pgfqpoint{0.643cm}{1.656cm}}{\pgfqpoint{0.679cm}{1.656cm}}
\pgfpathcurveto{\pgfqpoint{0.715cm}{1.656cm}}{\pgfqpoint{0.75cm}{1.67cm}}{\pgfqpoint{0.775cm}{1.696cm}}
\pgfpathcurveto{\pgfqpoint{0.801cm}{1.722cm}}{\pgfqpoint{0.815cm}{1.756cm}}{\pgfqpoint{0.815cm}{1.793cm}}
\pgfusepath{fill}
\pgfpathmoveto{\pgfqpoint{1.345cm}{1.765cm}}
\pgfpathcurveto{\pgfqpoint{1.345cm}{1.801cm}}{\pgfqpoint{1.331cm}{1.836cm}}{\pgfqpoint{1.305cm}{1.862cm}}
\pgfpathcurveto{\pgfqpoint{1.28cm}{1.887cm}}{\pgfqpoint{1.245cm}{1.902cm}}{\pgfqpoint{1.209cm}{1.902cm}}
\pgfpathcurveto{\pgfqpoint{1.172cm}{1.902cm}}{\pgfqpoint{1.138cm}{1.887cm}}{\pgfqpoint{1.112cm}{1.862cm}}
\pgfpathcurveto{\pgfqpoint{1.087cm}{1.836cm}}{\pgfqpoint{1.072cm}{1.801cm}}{\pgfqpoint{1.072cm}{1.765cm}}
\pgfpathcurveto{\pgfqpoint{1.072cm}{1.728cm}}{\pgfqpoint{1.087cm}{1.694cm}}{\pgfqpoint{1.112cm}{1.668cm}}
\pgfpathcurveto{\pgfqpoint{1.138cm}{1.642cm}}{\pgfqpoint{1.172cm}{1.628cm}}{\pgfqpoint{1.209cm}{1.628cm}}
\pgfpathcurveto{\pgfqpoint{1.245cm}{1.628cm}}{\pgfqpoint{1.28cm}{1.642cm}}{\pgfqpoint{1.305cm}{1.668cm}}
\pgfpathcurveto{\pgfqpoint{1.331cm}{1.694cm}}{\pgfqpoint{1.345cm}{1.728cm}}{\pgfqpoint{1.345cm}{1.765cm}}
\pgfusepath{fill}
\begin{pgfscope}
\pgfsetdash{}{0cm}
\pgfsetlinewidth{0.818mm}
\pgfsetroundcap
\pgfsetroundjoin
\pgfsetmiterlimit{7.0}
\pgfpathmoveto{\pgfqpoint{0.682cm}{1.065cm}}
\pgfpathlineto{\pgfqpoint{1.246cm}{0.315cm}}
\pgfpathlineto{\pgfqpoint{1.811cm}{1.065cm}}
\pgfusepath{stroke}
\end{pgfscope}
\pgfpathmoveto{\pgfqpoint{1.948cm}{1.065cm}}
\pgfpathcurveto{\pgfqpoint{1.948cm}{1.101cm}}{\pgfqpoint{1.933cm}{1.136cm}}{\pgfqpoint{1.907cm}{1.162cm}}
\pgfpathcurveto{\pgfqpoint{1.882cm}{1.187cm}}{\pgfqpoint{1.847cm}{1.202cm}}{\pgfqpoint{1.811cm}{1.202cm}}
\pgfpathcurveto{\pgfqpoint{1.775cm}{1.202cm}}{\pgfqpoint{1.74cm}{1.187cm}}{\pgfqpoint{1.714cm}{1.162cm}}
\pgfpathcurveto{\pgfqpoint{1.689cm}{1.136cm}}{\pgfqpoint{1.674cm}{1.101cm}}{\pgfqpoint{1.674cm}{1.065cm}}
\pgfpathcurveto{\pgfqpoint{1.674cm}{1.029cm}}{\pgfqpoint{1.689cm}{0.994cm}}{\pgfqpoint{1.714cm}{0.968cm}}
\pgfpathcurveto{\pgfqpoint{1.74cm}{0.942cm}}{\pgfqpoint{1.775cm}{0.928cm}}{\pgfqpoint{1.811cm}{0.928cm}}
\pgfpathcurveto{\pgfqpoint{1.847cm}{0.928cm}}{\pgfqpoint{1.882cm}{0.942cm}}{\pgfqpoint{1.907cm}{0.968cm}}
\pgfpathcurveto{\pgfqpoint{1.933cm}{0.994cm}}{\pgfqpoint{1.948cm}{1.029cm}}{\pgfqpoint{1.948cm}{1.065cm}}
\pgfusepath{fill}
\begin{pgfscope}
\pgfsetdash{}{0cm}
\pgfsetlinewidth{0.818mm}
\pgfsetmiterlimit{4.0}
\pgfpathmoveto{\pgfqpoint{1.383cm}{0.178cm}}
\pgfpathcurveto{\pgfqpoint{1.383cm}{0.214cm}}{\pgfqpoint{1.369cm}{0.249cm}}{\pgfqpoint{1.343cm}{0.275cm}}
\pgfpathcurveto{\pgfqpoint{1.317cm}{0.3cm}}{\pgfqpoint{1.283cm}{0.315cm}}{\pgfqpoint{1.246cm}{0.315cm}}
\pgfpathcurveto{\pgfqpoint{1.21cm}{0.315cm}}{\pgfqpoint{1.175cm}{0.3cm}}{\pgfqpoint{1.15cm}{0.275cm}}
\pgfpathcurveto{\pgfqpoint{1.124cm}{0.249cm}}{\pgfqpoint{1.11cm}{0.214cm}}{\pgfqpoint{1.11cm}{0.178cm}}
\pgfpathcurveto{\pgfqpoint{1.11cm}{0.141cm}}{\pgfqpoint{1.124cm}{0.107cm}}{\pgfqpoint{1.15cm}{0.081cm}}
\pgfpathcurveto{\pgfqpoint{1.175cm}{0.055cm}}{\pgfqpoint{1.21cm}{0.041cm}}{\pgfqpoint{1.246cm}{0.041cm}}
\pgfpathcurveto{\pgfqpoint{1.283cm}{0.041cm}}{\pgfqpoint{1.317cm}{0.055cm}}{\pgfqpoint{1.343cm}{0.081cm}}
\pgfpathcurveto{\pgfqpoint{1.369cm}{0.107cm}}{\pgfqpoint{1.383cm}{0.141cm}}{\pgfqpoint{1.383cm}{0.178cm}}
\pgfusepath{stroke}
\end{pgfscope}
\end{pgfscope}
\end{pgfscope}
\end{pgfscope}
\end{tikzpicture}}}}\\
&\qquad+\rmm{3\UU_>\llbracket X^2 \rrbracket\succ(\phi+\psi)}+\rmm{3\UU_>\llbracket X^2 \rrbracket\prec(\phi+\psi)}+\rmm{3X^{\!\resizebox{0.6em}{!}{
\begin{tikzpicture}
\pgfpathmoveto{\pgfqpoint{0cm}{-0.035cm}}
\pgfpathlineto{\pgfqpoint{1.376cm}{-0.035cm}}
\pgfpathlineto{\pgfqpoint{1.376cm}{1.552cm}}
\pgfpathlineto{\pgfqpoint{0cm}{1.552cm}}
\pgfpathclose
\pgfusepath{clip}
\begin{pgfscope}
\begin{pgfscope}
\pgfpathmoveto{\pgfqpoint{0cm}{-0.035cm}}
\pgfpathlineto{\pgfqpoint{1.376cm}{-0.035cm}}
\pgfpathlineto{\pgfqpoint{1.376cm}{1.552cm}}
\pgfpathlineto{\pgfqpoint{0cm}{1.552cm}}
\pgfpathclose
\pgfusepath{clip}
\begin{pgfscope}
\begin{pgfscope}
\pgfsetdash{}{0cm}
\pgfsetlinewidth{0.818mm}
\pgfsetroundcap
\pgfsetroundjoin
\pgfsetmiterlimit{7.0}
\definecolor{eps2pgf_color}{gray}{0}\pgfsetstrokecolor{eps2pgf_color}\pgfsetfillcolor{eps2pgf_color}
\pgfpathmoveto{\pgfqpoint{0.117cm}{1.421cm}}
\pgfpathlineto{\pgfqpoint{0.682cm}{0.671cm}}
\pgfpathlineto{\pgfqpoint{1.246cm}{1.421cm}}
\pgfusepath{stroke}
\end{pgfscope}
\definecolor{eps2pgf_color}{gray}{0}\pgfsetstrokecolor{eps2pgf_color}\pgfsetfillcolor{eps2pgf_color}
\pgfpathmoveto{\pgfqpoint{0.273cm}{1.395cm}}
\pgfpathcurveto{\pgfqpoint{0.273cm}{1.432cm}}{\pgfqpoint{0.259cm}{1.467cm}}{\pgfqpoint{0.233cm}{1.492cm}}
\pgfpathcurveto{\pgfqpoint{0.207cm}{1.518cm}}{\pgfqpoint{0.173cm}{1.532cm}}{\pgfqpoint{0.137cm}{1.532cm}}
\pgfpathcurveto{\pgfqpoint{0.1cm}{1.532cm}}{\pgfqpoint{0.066cm}{1.518cm}}{\pgfqpoint{0.04cm}{1.492cm}}
\pgfpathcurveto{\pgfqpoint{0.014cm}{1.467cm}}{\pgfqpoint{0cm}{1.432cm}}{\pgfqpoint{0cm}{1.395cm}}
\pgfpathcurveto{\pgfqpoint{0cm}{1.359cm}}{\pgfqpoint{0.014cm}{1.324cm}}{\pgfqpoint{0.04cm}{1.299cm}}
\pgfpathcurveto{\pgfqpoint{0.066cm}{1.273cm}}{\pgfqpoint{0.1cm}{1.258cm}}{\pgfqpoint{0.137cm}{1.258cm}}
\pgfpathcurveto{\pgfqpoint{0.173cm}{1.258cm}}{\pgfqpoint{0.207cm}{1.273cm}}{\pgfqpoint{0.233cm}{1.299cm}}
\pgfpathcurveto{\pgfqpoint{0.259cm}{1.324cm}}{\pgfqpoint{0.273cm}{1.359cm}}{\pgfqpoint{0.273cm}{1.395cm}}
\pgfusepath{fill}
\begin{pgfscope}
\pgfsetdash{}{0cm}
\pgfsetlinewidth{0.818mm}
\pgfsetmiterlimit{7.0}
\pgfpathmoveto{\pgfqpoint{0.682cm}{0.671cm}}
\pgfpathlineto{\pgfqpoint{0.679cm}{1.418cm}}
\pgfusepath{stroke}
\end{pgfscope}
\pgfpathmoveto{\pgfqpoint{0.815cm}{1.399cm}}
\pgfpathcurveto{\pgfqpoint{0.815cm}{1.435cm}}{\pgfqpoint{0.801cm}{1.47cm}}{\pgfqpoint{0.775cm}{1.496cm}}
\pgfpathcurveto{\pgfqpoint{0.75cm}{1.521cm}}{\pgfqpoint{0.715cm}{1.536cm}}{\pgfqpoint{0.679cm}{1.536cm}}
\pgfpathcurveto{\pgfqpoint{0.643cm}{1.536cm}}{\pgfqpoint{0.608cm}{1.521cm}}{\pgfqpoint{0.582cm}{1.496cm}}
\pgfpathcurveto{\pgfqpoint{0.557cm}{1.47cm}}{\pgfqpoint{0.542cm}{1.435cm}}{\pgfqpoint{0.542cm}{1.399cm}}
\pgfpathcurveto{\pgfqpoint{0.542cm}{1.363cm}}{\pgfqpoint{0.557cm}{1.328cm}}{\pgfqpoint{0.582cm}{1.302cm}}
\pgfpathcurveto{\pgfqpoint{0.608cm}{1.276cm}}{\pgfqpoint{0.643cm}{1.262cm}}{\pgfqpoint{0.679cm}{1.262cm}}
\pgfpathcurveto{\pgfqpoint{0.715cm}{1.262cm}}{\pgfqpoint{0.75cm}{1.276cm}}{\pgfqpoint{0.775cm}{1.302cm}}
\pgfpathcurveto{\pgfqpoint{0.801cm}{1.328cm}}{\pgfqpoint{0.815cm}{1.363cm}}{\pgfqpoint{0.815cm}{1.399cm}}
\pgfusepath{fill}
\pgfpathmoveto{\pgfqpoint{1.345cm}{1.371cm}}
\pgfpathcurveto{\pgfqpoint{1.345cm}{1.408cm}}{\pgfqpoint{1.331cm}{1.442cm}}{\pgfqpoint{1.305cm}{1.468cm}}
\pgfpathcurveto{\pgfqpoint{1.28cm}{1.494cm}}{\pgfqpoint{1.245cm}{1.508cm}}{\pgfqpoint{1.209cm}{1.508cm}}
\pgfpathcurveto{\pgfqpoint{1.172cm}{1.508cm}}{\pgfqpoint{1.138cm}{1.494cm}}{\pgfqpoint{1.112cm}{1.468cm}}
\pgfpathcurveto{\pgfqpoint{1.087cm}{1.442cm}}{\pgfqpoint{1.072cm}{1.408cm}}{\pgfqpoint{1.072cm}{1.371cm}}
\pgfpathcurveto{\pgfqpoint{1.072cm}{1.335cm}}{\pgfqpoint{1.087cm}{1.3cm}}{\pgfqpoint{1.112cm}{1.274cm}}
\pgfpathcurveto{\pgfqpoint{1.138cm}{1.249cm}}{\pgfqpoint{1.172cm}{1.234cm}}{\pgfqpoint{1.209cm}{1.234cm}}
\pgfpathcurveto{\pgfqpoint{1.245cm}{1.234cm}}{\pgfqpoint{1.28cm}{1.249cm}}{\pgfqpoint{1.305cm}{1.274cm}}
\pgfpathcurveto{\pgfqpoint{1.331cm}{1.3cm}}{\pgfqpoint{1.345cm}{1.335cm}}{\pgfqpoint{1.345cm}{1.371cm}}
\pgfusepath{fill}
\begin{pgfscope}
\pgfsetdash{}{0cm}
\pgfsetlinewidth{0.818mm}
\pgfsetroundcap
\pgfsetmiterlimit{4.0}
\pgfpathmoveto{\pgfqpoint{0.682cm}{0.671cm}}
\pgfpathlineto{\pgfqpoint{0.682cm}{0.042cm}}
\pgfusepath{stroke}
\end{pgfscope}
\end{pgfscope}
\end{pgfscope}
\end{pgfscope}
\end{tikzpicture}}}X^{\!\resizebox{!}{.8em}{
\begin{tikzpicture}
\pgfpathmoveto{\pgfqpoint{0cm}{-0.035cm}}
\pgfpathlineto{\pgfqpoint{1.976cm}{-0.035cm}}
\pgfpathlineto{\pgfqpoint{1.976cm}{1.94cm}}
\pgfpathlineto{\pgfqpoint{0cm}{1.94cm}}
\pgfpathclose
\pgfusepath{clip}
\begin{pgfscope}
\begin{pgfscope}
\pgfpathmoveto{\pgfqpoint{0cm}{-0.035cm}}
\pgfpathlineto{\pgfqpoint{1.976cm}{-0.035cm}}
\pgfpathlineto{\pgfqpoint{1.976cm}{1.94cm}}
\pgfpathlineto{\pgfqpoint{0cm}{1.94cm}}
\pgfpathclose
\pgfusepath{clip}
\begin{pgfscope}
\begin{pgfscope}
\pgfsetdash{}{0cm}
\pgfsetlinewidth{0.818mm}
\pgfsetroundcap
\pgfsetroundjoin
\pgfsetmiterlimit{7.0}
\definecolor{eps2pgf_color}{gray}{0}\pgfsetstrokecolor{eps2pgf_color}\pgfsetfillcolor{eps2pgf_color}
\pgfpathmoveto{\pgfqpoint{0.117cm}{1.815cm}}
\pgfpathlineto{\pgfqpoint{0.682cm}{1.065cm}}
\pgfpathlineto{\pgfqpoint{1.246cm}{1.815cm}}
\pgfusepath{stroke}
\end{pgfscope}
\definecolor{eps2pgf_color}{gray}{0}\pgfsetstrokecolor{eps2pgf_color}\pgfsetfillcolor{eps2pgf_color}
\pgfpathmoveto{\pgfqpoint{0.273cm}{1.789cm}}
\pgfpathcurveto{\pgfqpoint{0.273cm}{1.825cm}}{\pgfqpoint{0.259cm}{1.86cm}}{\pgfqpoint{0.233cm}{1.886cm}}
\pgfpathcurveto{\pgfqpoint{0.207cm}{1.912cm}}{\pgfqpoint{0.173cm}{1.926cm}}{\pgfqpoint{0.137cm}{1.926cm}}
\pgfpathcurveto{\pgfqpoint{0.1cm}{1.926cm}}{\pgfqpoint{0.066cm}{1.912cm}}{\pgfqpoint{0.04cm}{1.886cm}}
\pgfpathcurveto{\pgfqpoint{0.014cm}{1.86cm}}{\pgfqpoint{0cm}{1.825cm}}{\pgfqpoint{0cm}{1.789cm}}
\pgfpathcurveto{\pgfqpoint{0cm}{1.753cm}}{\pgfqpoint{0.014cm}{1.718cm}}{\pgfqpoint{0.04cm}{1.692cm}}
\pgfpathcurveto{\pgfqpoint{0.066cm}{1.667cm}}{\pgfqpoint{0.1cm}{1.652cm}}{\pgfqpoint{0.137cm}{1.652cm}}
\pgfpathcurveto{\pgfqpoint{0.173cm}{1.652cm}}{\pgfqpoint{0.207cm}{1.667cm}}{\pgfqpoint{0.233cm}{1.692cm}}
\pgfpathcurveto{\pgfqpoint{0.259cm}{1.718cm}}{\pgfqpoint{0.273cm}{1.753cm}}{\pgfqpoint{0.273cm}{1.789cm}}
\pgfusepath{fill}
\pgfpathmoveto{\pgfqpoint{1.345cm}{1.765cm}}
\pgfpathcurveto{\pgfqpoint{1.345cm}{1.801cm}}{\pgfqpoint{1.331cm}{1.836cm}}{\pgfqpoint{1.305cm}{1.862cm}}
\pgfpathcurveto{\pgfqpoint{1.28cm}{1.887cm}}{\pgfqpoint{1.245cm}{1.902cm}}{\pgfqpoint{1.209cm}{1.902cm}}
\pgfpathcurveto{\pgfqpoint{1.172cm}{1.902cm}}{\pgfqpoint{1.138cm}{1.887cm}}{\pgfqpoint{1.112cm}{1.862cm}}
\pgfpathcurveto{\pgfqpoint{1.087cm}{1.836cm}}{\pgfqpoint{1.072cm}{1.801cm}}{\pgfqpoint{1.072cm}{1.765cm}}
\pgfpathcurveto{\pgfqpoint{1.072cm}{1.728cm}}{\pgfqpoint{1.087cm}{1.694cm}}{\pgfqpoint{1.112cm}{1.668cm}}
\pgfpathcurveto{\pgfqpoint{1.138cm}{1.642cm}}{\pgfqpoint{1.172cm}{1.628cm}}{\pgfqpoint{1.209cm}{1.628cm}}
\pgfpathcurveto{\pgfqpoint{1.245cm}{1.628cm}}{\pgfqpoint{1.28cm}{1.642cm}}{\pgfqpoint{1.305cm}{1.668cm}}
\pgfpathcurveto{\pgfqpoint{1.331cm}{1.694cm}}{\pgfqpoint{1.345cm}{1.728cm}}{\pgfqpoint{1.345cm}{1.765cm}}
\pgfusepath{fill}
\begin{pgfscope}
\pgfsetdash{}{0cm}
\pgfsetlinewidth{0.818mm}
\pgfsetroundcap
\pgfsetroundjoin
\pgfsetmiterlimit{7.0}
\pgfpathmoveto{\pgfqpoint{0.682cm}{1.065cm}}
\pgfpathlineto{\pgfqpoint{1.246cm}{0.315cm}}
\pgfpathlineto{\pgfqpoint{1.811cm}{1.065cm}}
\pgfusepath{stroke}
\end{pgfscope}
\pgfpathmoveto{\pgfqpoint{1.948cm}{1.065cm}}
\pgfpathcurveto{\pgfqpoint{1.948cm}{1.101cm}}{\pgfqpoint{1.933cm}{1.136cm}}{\pgfqpoint{1.907cm}{1.162cm}}
\pgfpathcurveto{\pgfqpoint{1.882cm}{1.187cm}}{\pgfqpoint{1.847cm}{1.202cm}}{\pgfqpoint{1.811cm}{1.202cm}}
\pgfpathcurveto{\pgfqpoint{1.775cm}{1.202cm}}{\pgfqpoint{1.74cm}{1.187cm}}{\pgfqpoint{1.714cm}{1.162cm}}
\pgfpathcurveto{\pgfqpoint{1.689cm}{1.136cm}}{\pgfqpoint{1.674cm}{1.101cm}}{\pgfqpoint{1.674cm}{1.065cm}}
\pgfpathcurveto{\pgfqpoint{1.674cm}{1.029cm}}{\pgfqpoint{1.689cm}{0.994cm}}{\pgfqpoint{1.714cm}{0.968cm}}
\pgfpathcurveto{\pgfqpoint{1.74cm}{0.942cm}}{\pgfqpoint{1.775cm}{0.928cm}}{\pgfqpoint{1.811cm}{0.928cm}}
\pgfpathcurveto{\pgfqpoint{1.847cm}{0.928cm}}{\pgfqpoint{1.882cm}{0.942cm}}{\pgfqpoint{1.907cm}{0.968cm}}
\pgfpathcurveto{\pgfqpoint{1.933cm}{0.994cm}}{\pgfqpoint{1.948cm}{1.029cm}}{\pgfqpoint{1.948cm}{1.065cm}}
\pgfusepath{fill}
\begin{pgfscope}
\pgfsetdash{}{0cm}
\pgfsetlinewidth{0.818mm}
\pgfsetmiterlimit{7.0}
\pgfpathmoveto{\pgfqpoint{1.246cm}{0.315cm}}
\pgfpathlineto{\pgfqpoint{1.244cm}{1.061cm}}
\pgfusepath{stroke}
\end{pgfscope}
\pgfpathmoveto{\pgfqpoint{1.38cm}{1.065cm}}
\pgfpathcurveto{\pgfqpoint{1.38cm}{1.101cm}}{\pgfqpoint{1.366cm}{1.136cm}}{\pgfqpoint{1.34cm}{1.162cm}}
\pgfpathcurveto{\pgfqpoint{1.315cm}{1.187cm}}{\pgfqpoint{1.28cm}{1.202cm}}{\pgfqpoint{1.244cm}{1.202cm}}
\pgfpathcurveto{\pgfqpoint{1.207cm}{1.202cm}}{\pgfqpoint{1.173cm}{1.187cm}}{\pgfqpoint{1.147cm}{1.162cm}}
\pgfpathcurveto{\pgfqpoint{1.121cm}{1.136cm}}{\pgfqpoint{1.107cm}{1.101cm}}{\pgfqpoint{1.107cm}{1.065cm}}
\pgfpathcurveto{\pgfqpoint{1.107cm}{1.029cm}}{\pgfqpoint{1.121cm}{0.994cm}}{\pgfqpoint{1.147cm}{0.968cm}}
\pgfpathcurveto{\pgfqpoint{1.173cm}{0.942cm}}{\pgfqpoint{1.207cm}{0.928cm}}{\pgfqpoint{1.244cm}{0.928cm}}
\pgfpathcurveto{\pgfqpoint{1.28cm}{0.928cm}}{\pgfqpoint{1.315cm}{0.942cm}}{\pgfqpoint{1.34cm}{0.968cm}}
\pgfpathcurveto{\pgfqpoint{1.366cm}{0.994cm}}{\pgfqpoint{1.38cm}{1.029cm}}{\pgfqpoint{1.38cm}{1.065cm}}
\pgfusepath{fill}
\begin{pgfscope}
\pgfsetdash{}{0cm}
\pgfsetlinewidth{0.818mm}
\pgfsetmiterlimit{4.0}
\pgfpathmoveto{\pgfqpoint{1.383cm}{0.178cm}}
\pgfpathcurveto{\pgfqpoint{1.383cm}{0.214cm}}{\pgfqpoint{1.369cm}{0.249cm}}{\pgfqpoint{1.343cm}{0.275cm}}
\pgfpathcurveto{\pgfqpoint{1.317cm}{0.3cm}}{\pgfqpoint{1.283cm}{0.315cm}}{\pgfqpoint{1.246cm}{0.315cm}}
\pgfpathcurveto{\pgfqpoint{1.21cm}{0.315cm}}{\pgfqpoint{1.175cm}{0.3cm}}{\pgfqpoint{1.15cm}{0.275cm}}
\pgfpathcurveto{\pgfqpoint{1.124cm}{0.249cm}}{\pgfqpoint{1.11cm}{0.214cm}}{\pgfqpoint{1.11cm}{0.178cm}}
\pgfpathcurveto{\pgfqpoint{1.11cm}{0.141cm}}{\pgfqpoint{1.124cm}{0.107cm}}{\pgfqpoint{1.15cm}{0.081cm}}
\pgfpathcurveto{\pgfqpoint{1.175cm}{0.055cm}}{\pgfqpoint{1.21cm}{0.041cm}}{\pgfqpoint{1.246cm}{0.041cm}}
\pgfpathcurveto{\pgfqpoint{1.283cm}{0.041cm}}{\pgfqpoint{1.317cm}{0.055cm}}{\pgfqpoint{1.343cm}{0.081cm}}
\pgfpathcurveto{\pgfqpoint{1.369cm}{0.107cm}}{\pgfqpoint{1.383cm}{0.141cm}}{\pgfqpoint{1.383cm}{0.178cm}}
\pgfusepath{stroke}
\end{pgfscope}
\end{pgfscope}
\end{pgfscope}
\end{pgfscope}
\end{tikzpicture}}}}-\rmm{3( \phi + \psi)\prec\UU_>X^{\!\resizebox{!}{.8em}{
\begin{tikzpicture}
\pgfpathmoveto{\pgfqpoint{0cm}{-0.035cm}}
\pgfpathlineto{\pgfqpoint{1.976cm}{-0.035cm}}
\pgfpathlineto{\pgfqpoint{1.976cm}{1.94cm}}
\pgfpathlineto{\pgfqpoint{0cm}{1.94cm}}
\pgfpathclose
\pgfusepath{clip}
\begin{pgfscope}
\begin{pgfscope}
\pgfpathmoveto{\pgfqpoint{0cm}{-0.035cm}}
\pgfpathlineto{\pgfqpoint{1.976cm}{-0.035cm}}
\pgfpathlineto{\pgfqpoint{1.976cm}{1.94cm}}
\pgfpathlineto{\pgfqpoint{0cm}{1.94cm}}
\pgfpathclose
\pgfusepath{clip}
\begin{pgfscope}
\begin{pgfscope}
\pgfsetdash{}{0cm}
\pgfsetlinewidth{0.818mm}
\pgfsetroundcap
\pgfsetroundjoin
\pgfsetmiterlimit{7.0}
\definecolor{eps2pgf_color}{gray}{0}\pgfsetstrokecolor{eps2pgf_color}\pgfsetfillcolor{eps2pgf_color}
\pgfpathmoveto{\pgfqpoint{0.117cm}{1.815cm}}
\pgfpathlineto{\pgfqpoint{0.682cm}{1.065cm}}
\pgfpathlineto{\pgfqpoint{1.246cm}{1.815cm}}
\pgfusepath{stroke}
\end{pgfscope}
\definecolor{eps2pgf_color}{gray}{0}\pgfsetstrokecolor{eps2pgf_color}\pgfsetfillcolor{eps2pgf_color}
\pgfpathmoveto{\pgfqpoint{0.273cm}{1.789cm}}
\pgfpathcurveto{\pgfqpoint{0.273cm}{1.825cm}}{\pgfqpoint{0.259cm}{1.86cm}}{\pgfqpoint{0.233cm}{1.886cm}}
\pgfpathcurveto{\pgfqpoint{0.207cm}{1.912cm}}{\pgfqpoint{0.173cm}{1.926cm}}{\pgfqpoint{0.137cm}{1.926cm}}
\pgfpathcurveto{\pgfqpoint{0.1cm}{1.926cm}}{\pgfqpoint{0.066cm}{1.912cm}}{\pgfqpoint{0.04cm}{1.886cm}}
\pgfpathcurveto{\pgfqpoint{0.014cm}{1.86cm}}{\pgfqpoint{0cm}{1.825cm}}{\pgfqpoint{0cm}{1.789cm}}
\pgfpathcurveto{\pgfqpoint{0cm}{1.753cm}}{\pgfqpoint{0.014cm}{1.718cm}}{\pgfqpoint{0.04cm}{1.692cm}}
\pgfpathcurveto{\pgfqpoint{0.066cm}{1.667cm}}{\pgfqpoint{0.1cm}{1.652cm}}{\pgfqpoint{0.137cm}{1.652cm}}
\pgfpathcurveto{\pgfqpoint{0.173cm}{1.652cm}}{\pgfqpoint{0.207cm}{1.667cm}}{\pgfqpoint{0.233cm}{1.692cm}}
\pgfpathcurveto{\pgfqpoint{0.259cm}{1.718cm}}{\pgfqpoint{0.273cm}{1.753cm}}{\pgfqpoint{0.273cm}{1.789cm}}
\pgfusepath{fill}
\pgfpathmoveto{\pgfqpoint{1.345cm}{1.765cm}}
\pgfpathcurveto{\pgfqpoint{1.345cm}{1.801cm}}{\pgfqpoint{1.331cm}{1.836cm}}{\pgfqpoint{1.305cm}{1.862cm}}
\pgfpathcurveto{\pgfqpoint{1.28cm}{1.887cm}}{\pgfqpoint{1.245cm}{1.902cm}}{\pgfqpoint{1.209cm}{1.902cm}}
\pgfpathcurveto{\pgfqpoint{1.172cm}{1.902cm}}{\pgfqpoint{1.138cm}{1.887cm}}{\pgfqpoint{1.112cm}{1.862cm}}
\pgfpathcurveto{\pgfqpoint{1.087cm}{1.836cm}}{\pgfqpoint{1.072cm}{1.801cm}}{\pgfqpoint{1.072cm}{1.765cm}}
\pgfpathcurveto{\pgfqpoint{1.072cm}{1.728cm}}{\pgfqpoint{1.087cm}{1.694cm}}{\pgfqpoint{1.112cm}{1.668cm}}
\pgfpathcurveto{\pgfqpoint{1.138cm}{1.642cm}}{\pgfqpoint{1.172cm}{1.628cm}}{\pgfqpoint{1.209cm}{1.628cm}}
\pgfpathcurveto{\pgfqpoint{1.245cm}{1.628cm}}{\pgfqpoint{1.28cm}{1.642cm}}{\pgfqpoint{1.305cm}{1.668cm}}
\pgfpathcurveto{\pgfqpoint{1.331cm}{1.694cm}}{\pgfqpoint{1.345cm}{1.728cm}}{\pgfqpoint{1.345cm}{1.765cm}}
\pgfusepath{fill}
\begin{pgfscope}
\pgfsetdash{}{0cm}
\pgfsetlinewidth{0.818mm}
\pgfsetroundcap
\pgfsetroundjoin
\pgfsetmiterlimit{7.0}
\pgfpathmoveto{\pgfqpoint{0.682cm}{1.065cm}}
\pgfpathlineto{\pgfqpoint{1.246cm}{0.315cm}}
\pgfpathlineto{\pgfqpoint{1.811cm}{1.065cm}}
\pgfusepath{stroke}
\end{pgfscope}
\pgfpathmoveto{\pgfqpoint{1.948cm}{1.065cm}}
\pgfpathcurveto{\pgfqpoint{1.948cm}{1.101cm}}{\pgfqpoint{1.933cm}{1.136cm}}{\pgfqpoint{1.907cm}{1.162cm}}
\pgfpathcurveto{\pgfqpoint{1.882cm}{1.187cm}}{\pgfqpoint{1.847cm}{1.202cm}}{\pgfqpoint{1.811cm}{1.202cm}}
\pgfpathcurveto{\pgfqpoint{1.775cm}{1.202cm}}{\pgfqpoint{1.74cm}{1.187cm}}{\pgfqpoint{1.714cm}{1.162cm}}
\pgfpathcurveto{\pgfqpoint{1.689cm}{1.136cm}}{\pgfqpoint{1.674cm}{1.101cm}}{\pgfqpoint{1.674cm}{1.065cm}}
\pgfpathcurveto{\pgfqpoint{1.674cm}{1.029cm}}{\pgfqpoint{1.689cm}{0.994cm}}{\pgfqpoint{1.714cm}{0.968cm}}
\pgfpathcurveto{\pgfqpoint{1.74cm}{0.942cm}}{\pgfqpoint{1.775cm}{0.928cm}}{\pgfqpoint{1.811cm}{0.928cm}}
\pgfpathcurveto{\pgfqpoint{1.847cm}{0.928cm}}{\pgfqpoint{1.882cm}{0.942cm}}{\pgfqpoint{1.907cm}{0.968cm}}
\pgfpathcurveto{\pgfqpoint{1.933cm}{0.994cm}}{\pgfqpoint{1.948cm}{1.029cm}}{\pgfqpoint{1.948cm}{1.065cm}}
\pgfusepath{fill}
\begin{pgfscope}
\pgfsetdash{}{0cm}
\pgfsetlinewidth{0.818mm}
\pgfsetmiterlimit{7.0}
\pgfpathmoveto{\pgfqpoint{1.246cm}{0.315cm}}
\pgfpathlineto{\pgfqpoint{1.244cm}{1.061cm}}
\pgfusepath{stroke}
\end{pgfscope}
\pgfpathmoveto{\pgfqpoint{1.38cm}{1.065cm}}
\pgfpathcurveto{\pgfqpoint{1.38cm}{1.101cm}}{\pgfqpoint{1.366cm}{1.136cm}}{\pgfqpoint{1.34cm}{1.162cm}}
\pgfpathcurveto{\pgfqpoint{1.315cm}{1.187cm}}{\pgfqpoint{1.28cm}{1.202cm}}{\pgfqpoint{1.244cm}{1.202cm}}
\pgfpathcurveto{\pgfqpoint{1.207cm}{1.202cm}}{\pgfqpoint{1.173cm}{1.187cm}}{\pgfqpoint{1.147cm}{1.162cm}}
\pgfpathcurveto{\pgfqpoint{1.121cm}{1.136cm}}{\pgfqpoint{1.107cm}{1.101cm}}{\pgfqpoint{1.107cm}{1.065cm}}
\pgfpathcurveto{\pgfqpoint{1.107cm}{1.029cm}}{\pgfqpoint{1.121cm}{0.994cm}}{\pgfqpoint{1.147cm}{0.968cm}}
\pgfpathcurveto{\pgfqpoint{1.173cm}{0.942cm}}{\pgfqpoint{1.207cm}{0.928cm}}{\pgfqpoint{1.244cm}{0.928cm}}
\pgfpathcurveto{\pgfqpoint{1.28cm}{0.928cm}}{\pgfqpoint{1.315cm}{0.942cm}}{\pgfqpoint{1.34cm}{0.968cm}}
\pgfpathcurveto{\pgfqpoint{1.366cm}{0.994cm}}{\pgfqpoint{1.38cm}{1.029cm}}{\pgfqpoint{1.38cm}{1.065cm}}
\pgfusepath{fill}
\begin{pgfscope}
\pgfsetdash{}{0cm}
\pgfsetlinewidth{0.818mm}
\pgfsetmiterlimit{4.0}
\pgfpathmoveto{\pgfqpoint{1.383cm}{0.178cm}}
\pgfpathcurveto{\pgfqpoint{1.383cm}{0.214cm}}{\pgfqpoint{1.369cm}{0.249cm}}{\pgfqpoint{1.343cm}{0.275cm}}
\pgfpathcurveto{\pgfqpoint{1.317cm}{0.3cm}}{\pgfqpoint{1.283cm}{0.315cm}}{\pgfqpoint{1.246cm}{0.315cm}}
\pgfpathcurveto{\pgfqpoint{1.21cm}{0.315cm}}{\pgfqpoint{1.175cm}{0.3cm}}{\pgfqpoint{1.15cm}{0.275cm}}
\pgfpathcurveto{\pgfqpoint{1.124cm}{0.249cm}}{\pgfqpoint{1.11cm}{0.214cm}}{\pgfqpoint{1.11cm}{0.178cm}}
\pgfpathcurveto{\pgfqpoint{1.11cm}{0.141cm}}{\pgfqpoint{1.124cm}{0.107cm}}{\pgfqpoint{1.15cm}{0.081cm}}
\pgfpathcurveto{\pgfqpoint{1.175cm}{0.055cm}}{\pgfqpoint{1.21cm}{0.041cm}}{\pgfqpoint{1.246cm}{0.041cm}}
\pgfpathcurveto{\pgfqpoint{1.283cm}{0.041cm}}{\pgfqpoint{1.317cm}{0.055cm}}{\pgfqpoint{1.343cm}{0.081cm}}
\pgfpathcurveto{\pgfqpoint{1.369cm}{0.107cm}}{\pgfqpoint{1.383cm}{0.141cm}}{\pgfqpoint{1.383cm}{0.178cm}}
\pgfusepath{stroke}
\end{pgfscope}
\end{pgfscope}
\end{pgfscope}
\end{pgfscope}
\end{tikzpicture}}}}\\
&\qquad+\rmm{6\UU_> X\prec(X^{\!\resizebox{0.6em}{!}{
\begin{tikzpicture}
\pgfpathmoveto{\pgfqpoint{0cm}{-0.035cm}}
\pgfpathlineto{\pgfqpoint{1.376cm}{-0.035cm}}
\pgfpathlineto{\pgfqpoint{1.376cm}{1.552cm}}
\pgfpathlineto{\pgfqpoint{0cm}{1.552cm}}
\pgfpathclose
\pgfusepath{clip}
\begin{pgfscope}
\begin{pgfscope}
\pgfpathmoveto{\pgfqpoint{0cm}{-0.035cm}}
\pgfpathlineto{\pgfqpoint{1.376cm}{-0.035cm}}
\pgfpathlineto{\pgfqpoint{1.376cm}{1.552cm}}
\pgfpathlineto{\pgfqpoint{0cm}{1.552cm}}
\pgfpathclose
\pgfusepath{clip}
\begin{pgfscope}
\begin{pgfscope}
\pgfsetdash{}{0cm}
\pgfsetlinewidth{0.818mm}
\pgfsetroundcap
\pgfsetroundjoin
\pgfsetmiterlimit{7.0}
\definecolor{eps2pgf_color}{gray}{0}\pgfsetstrokecolor{eps2pgf_color}\pgfsetfillcolor{eps2pgf_color}
\pgfpathmoveto{\pgfqpoint{0.117cm}{1.421cm}}
\pgfpathlineto{\pgfqpoint{0.682cm}{0.671cm}}
\pgfpathlineto{\pgfqpoint{1.246cm}{1.421cm}}
\pgfusepath{stroke}
\end{pgfscope}
\definecolor{eps2pgf_color}{gray}{0}\pgfsetstrokecolor{eps2pgf_color}\pgfsetfillcolor{eps2pgf_color}
\pgfpathmoveto{\pgfqpoint{0.273cm}{1.395cm}}
\pgfpathcurveto{\pgfqpoint{0.273cm}{1.432cm}}{\pgfqpoint{0.259cm}{1.467cm}}{\pgfqpoint{0.233cm}{1.492cm}}
\pgfpathcurveto{\pgfqpoint{0.207cm}{1.518cm}}{\pgfqpoint{0.173cm}{1.532cm}}{\pgfqpoint{0.137cm}{1.532cm}}
\pgfpathcurveto{\pgfqpoint{0.1cm}{1.532cm}}{\pgfqpoint{0.066cm}{1.518cm}}{\pgfqpoint{0.04cm}{1.492cm}}
\pgfpathcurveto{\pgfqpoint{0.014cm}{1.467cm}}{\pgfqpoint{0cm}{1.432cm}}{\pgfqpoint{0cm}{1.395cm}}
\pgfpathcurveto{\pgfqpoint{0cm}{1.359cm}}{\pgfqpoint{0.014cm}{1.324cm}}{\pgfqpoint{0.04cm}{1.299cm}}
\pgfpathcurveto{\pgfqpoint{0.066cm}{1.273cm}}{\pgfqpoint{0.1cm}{1.258cm}}{\pgfqpoint{0.137cm}{1.258cm}}
\pgfpathcurveto{\pgfqpoint{0.173cm}{1.258cm}}{\pgfqpoint{0.207cm}{1.273cm}}{\pgfqpoint{0.233cm}{1.299cm}}
\pgfpathcurveto{\pgfqpoint{0.259cm}{1.324cm}}{\pgfqpoint{0.273cm}{1.359cm}}{\pgfqpoint{0.273cm}{1.395cm}}
\pgfusepath{fill}
\begin{pgfscope}
\pgfsetdash{}{0cm}
\pgfsetlinewidth{0.818mm}
\pgfsetmiterlimit{7.0}
\pgfpathmoveto{\pgfqpoint{0.682cm}{0.671cm}}
\pgfpathlineto{\pgfqpoint{0.679cm}{1.418cm}}
\pgfusepath{stroke}
\end{pgfscope}
\pgfpathmoveto{\pgfqpoint{0.815cm}{1.399cm}}
\pgfpathcurveto{\pgfqpoint{0.815cm}{1.435cm}}{\pgfqpoint{0.801cm}{1.47cm}}{\pgfqpoint{0.775cm}{1.496cm}}
\pgfpathcurveto{\pgfqpoint{0.75cm}{1.521cm}}{\pgfqpoint{0.715cm}{1.536cm}}{\pgfqpoint{0.679cm}{1.536cm}}
\pgfpathcurveto{\pgfqpoint{0.643cm}{1.536cm}}{\pgfqpoint{0.608cm}{1.521cm}}{\pgfqpoint{0.582cm}{1.496cm}}
\pgfpathcurveto{\pgfqpoint{0.557cm}{1.47cm}}{\pgfqpoint{0.542cm}{1.435cm}}{\pgfqpoint{0.542cm}{1.399cm}}
\pgfpathcurveto{\pgfqpoint{0.542cm}{1.363cm}}{\pgfqpoint{0.557cm}{1.328cm}}{\pgfqpoint{0.582cm}{1.302cm}}
\pgfpathcurveto{\pgfqpoint{0.608cm}{1.276cm}}{\pgfqpoint{0.643cm}{1.262cm}}{\pgfqpoint{0.679cm}{1.262cm}}
\pgfpathcurveto{\pgfqpoint{0.715cm}{1.262cm}}{\pgfqpoint{0.75cm}{1.276cm}}{\pgfqpoint{0.775cm}{1.302cm}}
\pgfpathcurveto{\pgfqpoint{0.801cm}{1.328cm}}{\pgfqpoint{0.815cm}{1.363cm}}{\pgfqpoint{0.815cm}{1.399cm}}
\pgfusepath{fill}
\pgfpathmoveto{\pgfqpoint{1.345cm}{1.371cm}}
\pgfpathcurveto{\pgfqpoint{1.345cm}{1.408cm}}{\pgfqpoint{1.331cm}{1.442cm}}{\pgfqpoint{1.305cm}{1.468cm}}
\pgfpathcurveto{\pgfqpoint{1.28cm}{1.494cm}}{\pgfqpoint{1.245cm}{1.508cm}}{\pgfqpoint{1.209cm}{1.508cm}}
\pgfpathcurveto{\pgfqpoint{1.172cm}{1.508cm}}{\pgfqpoint{1.138cm}{1.494cm}}{\pgfqpoint{1.112cm}{1.468cm}}
\pgfpathcurveto{\pgfqpoint{1.087cm}{1.442cm}}{\pgfqpoint{1.072cm}{1.408cm}}{\pgfqpoint{1.072cm}{1.371cm}}
\pgfpathcurveto{\pgfqpoint{1.072cm}{1.335cm}}{\pgfqpoint{1.087cm}{1.3cm}}{\pgfqpoint{1.112cm}{1.274cm}}
\pgfpathcurveto{\pgfqpoint{1.138cm}{1.249cm}}{\pgfqpoint{1.172cm}{1.234cm}}{\pgfqpoint{1.209cm}{1.234cm}}
\pgfpathcurveto{\pgfqpoint{1.245cm}{1.234cm}}{\pgfqpoint{1.28cm}{1.249cm}}{\pgfqpoint{1.305cm}{1.274cm}}
\pgfpathcurveto{\pgfqpoint{1.331cm}{1.3cm}}{\pgfqpoint{1.345cm}{1.335cm}}{\pgfqpoint{1.345cm}{1.371cm}}
\pgfusepath{fill}
\begin{pgfscope}
\pgfsetdash{}{0cm}
\pgfsetlinewidth{0.818mm}
\pgfsetroundcap
\pgfsetmiterlimit{4.0}
\pgfpathmoveto{\pgfqpoint{0.682cm}{0.671cm}}
\pgfpathlineto{\pgfqpoint{0.682cm}{0.042cm}}
\pgfusepath{stroke}
\end{pgfscope}
\end{pgfscope}
\end{pgfscope}
\end{pgfscope}
\end{tikzpicture}}}(\phi+\psi))}+\rmm{6\UU_> X\prec(X^{\!\resizebox{0.6em}{!}{
\begin{tikzpicture}
\pgfpathmoveto{\pgfqpoint{0cm}{-0.035cm}}
\pgfpathlineto{\pgfqpoint{1.376cm}{-0.035cm}}
\pgfpathlineto{\pgfqpoint{1.376cm}{1.552cm}}
\pgfpathlineto{\pgfqpoint{0cm}{1.552cm}}
\pgfpathclose
\pgfusepath{clip}
\begin{pgfscope}
\begin{pgfscope}
\pgfpathmoveto{\pgfqpoint{0cm}{-0.035cm}}
\pgfpathlineto{\pgfqpoint{1.376cm}{-0.035cm}}
\pgfpathlineto{\pgfqpoint{1.376cm}{1.552cm}}
\pgfpathlineto{\pgfqpoint{0cm}{1.552cm}}
\pgfpathclose
\pgfusepath{clip}
\begin{pgfscope}
\begin{pgfscope}
\pgfsetdash{}{0cm}
\pgfsetlinewidth{0.818mm}
\pgfsetroundcap
\pgfsetroundjoin
\pgfsetmiterlimit{7.0}
\definecolor{eps2pgf_color}{gray}{0}\pgfsetstrokecolor{eps2pgf_color}\pgfsetfillcolor{eps2pgf_color}
\pgfpathmoveto{\pgfqpoint{0.117cm}{1.421cm}}
\pgfpathlineto{\pgfqpoint{0.682cm}{0.671cm}}
\pgfpathlineto{\pgfqpoint{1.246cm}{1.421cm}}
\pgfusepath{stroke}
\end{pgfscope}
\definecolor{eps2pgf_color}{gray}{0}\pgfsetstrokecolor{eps2pgf_color}\pgfsetfillcolor{eps2pgf_color}
\pgfpathmoveto{\pgfqpoint{0.273cm}{1.395cm}}
\pgfpathcurveto{\pgfqpoint{0.273cm}{1.432cm}}{\pgfqpoint{0.259cm}{1.467cm}}{\pgfqpoint{0.233cm}{1.492cm}}
\pgfpathcurveto{\pgfqpoint{0.207cm}{1.518cm}}{\pgfqpoint{0.173cm}{1.532cm}}{\pgfqpoint{0.137cm}{1.532cm}}
\pgfpathcurveto{\pgfqpoint{0.1cm}{1.532cm}}{\pgfqpoint{0.066cm}{1.518cm}}{\pgfqpoint{0.04cm}{1.492cm}}
\pgfpathcurveto{\pgfqpoint{0.014cm}{1.467cm}}{\pgfqpoint{0cm}{1.432cm}}{\pgfqpoint{0cm}{1.395cm}}
\pgfpathcurveto{\pgfqpoint{0cm}{1.359cm}}{\pgfqpoint{0.014cm}{1.324cm}}{\pgfqpoint{0.04cm}{1.299cm}}
\pgfpathcurveto{\pgfqpoint{0.066cm}{1.273cm}}{\pgfqpoint{0.1cm}{1.258cm}}{\pgfqpoint{0.137cm}{1.258cm}}
\pgfpathcurveto{\pgfqpoint{0.173cm}{1.258cm}}{\pgfqpoint{0.207cm}{1.273cm}}{\pgfqpoint{0.233cm}{1.299cm}}
\pgfpathcurveto{\pgfqpoint{0.259cm}{1.324cm}}{\pgfqpoint{0.273cm}{1.359cm}}{\pgfqpoint{0.273cm}{1.395cm}}
\pgfusepath{fill}
\begin{pgfscope}
\pgfsetdash{}{0cm}
\pgfsetlinewidth{0.818mm}
\pgfsetmiterlimit{7.0}
\pgfpathmoveto{\pgfqpoint{0.682cm}{0.671cm}}
\pgfpathlineto{\pgfqpoint{0.679cm}{1.418cm}}
\pgfusepath{stroke}
\end{pgfscope}
\pgfpathmoveto{\pgfqpoint{0.815cm}{1.399cm}}
\pgfpathcurveto{\pgfqpoint{0.815cm}{1.435cm}}{\pgfqpoint{0.801cm}{1.47cm}}{\pgfqpoint{0.775cm}{1.496cm}}
\pgfpathcurveto{\pgfqpoint{0.75cm}{1.521cm}}{\pgfqpoint{0.715cm}{1.536cm}}{\pgfqpoint{0.679cm}{1.536cm}}
\pgfpathcurveto{\pgfqpoint{0.643cm}{1.536cm}}{\pgfqpoint{0.608cm}{1.521cm}}{\pgfqpoint{0.582cm}{1.496cm}}
\pgfpathcurveto{\pgfqpoint{0.557cm}{1.47cm}}{\pgfqpoint{0.542cm}{1.435cm}}{\pgfqpoint{0.542cm}{1.399cm}}
\pgfpathcurveto{\pgfqpoint{0.542cm}{1.363cm}}{\pgfqpoint{0.557cm}{1.328cm}}{\pgfqpoint{0.582cm}{1.302cm}}
\pgfpathcurveto{\pgfqpoint{0.608cm}{1.276cm}}{\pgfqpoint{0.643cm}{1.262cm}}{\pgfqpoint{0.679cm}{1.262cm}}
\pgfpathcurveto{\pgfqpoint{0.715cm}{1.262cm}}{\pgfqpoint{0.75cm}{1.276cm}}{\pgfqpoint{0.775cm}{1.302cm}}
\pgfpathcurveto{\pgfqpoint{0.801cm}{1.328cm}}{\pgfqpoint{0.815cm}{1.363cm}}{\pgfqpoint{0.815cm}{1.399cm}}
\pgfusepath{fill}
\pgfpathmoveto{\pgfqpoint{1.345cm}{1.371cm}}
\pgfpathcurveto{\pgfqpoint{1.345cm}{1.408cm}}{\pgfqpoint{1.331cm}{1.442cm}}{\pgfqpoint{1.305cm}{1.468cm}}
\pgfpathcurveto{\pgfqpoint{1.28cm}{1.494cm}}{\pgfqpoint{1.245cm}{1.508cm}}{\pgfqpoint{1.209cm}{1.508cm}}
\pgfpathcurveto{\pgfqpoint{1.172cm}{1.508cm}}{\pgfqpoint{1.138cm}{1.494cm}}{\pgfqpoint{1.112cm}{1.468cm}}
\pgfpathcurveto{\pgfqpoint{1.087cm}{1.442cm}}{\pgfqpoint{1.072cm}{1.408cm}}{\pgfqpoint{1.072cm}{1.371cm}}
\pgfpathcurveto{\pgfqpoint{1.072cm}{1.335cm}}{\pgfqpoint{1.087cm}{1.3cm}}{\pgfqpoint{1.112cm}{1.274cm}}
\pgfpathcurveto{\pgfqpoint{1.138cm}{1.249cm}}{\pgfqpoint{1.172cm}{1.234cm}}{\pgfqpoint{1.209cm}{1.234cm}}
\pgfpathcurveto{\pgfqpoint{1.245cm}{1.234cm}}{\pgfqpoint{1.28cm}{1.249cm}}{\pgfqpoint{1.305cm}{1.274cm}}
\pgfpathcurveto{\pgfqpoint{1.331cm}{1.3cm}}{\pgfqpoint{1.345cm}{1.335cm}}{\pgfqpoint{1.345cm}{1.371cm}}
\pgfusepath{fill}
\begin{pgfscope}
\pgfsetdash{}{0cm}
\pgfsetlinewidth{0.818mm}
\pgfsetroundcap
\pgfsetmiterlimit{4.0}
\pgfpathmoveto{\pgfqpoint{0.682cm}{0.671cm}}
\pgfpathlineto{\pgfqpoint{0.682cm}{0.042cm}}
\pgfusepath{stroke}
\end{pgfscope}
\end{pgfscope}
\end{pgfscope}
\end{pgfscope}
\end{tikzpicture}}}(\phi+\psi))}+\rmm{6(\phi+\psi)\prec\UU_{>}X^{\!\resizebox{!}{.8em}{
\begin{tikzpicture}
\pgfpathmoveto{\pgfqpoint{0cm}{-0.035cm}}
\pgfpathlineto{\pgfqpoint{1.976cm}{-0.035cm}}
\pgfpathlineto{\pgfqpoint{1.976cm}{1.94cm}}
\pgfpathlineto{\pgfqpoint{0cm}{1.94cm}}
\pgfpathclose
\pgfusepath{clip}
\begin{pgfscope}
\begin{pgfscope}
\pgfpathmoveto{\pgfqpoint{0cm}{-0.035cm}}
\pgfpathlineto{\pgfqpoint{1.976cm}{-0.035cm}}
\pgfpathlineto{\pgfqpoint{1.976cm}{1.94cm}}
\pgfpathlineto{\pgfqpoint{0cm}{1.94cm}}
\pgfpathclose
\pgfusepath{clip}
\begin{pgfscope}
\begin{pgfscope}
\pgfsetdash{}{0cm}
\pgfsetlinewidth{0.818mm}
\pgfsetroundcap
\pgfsetroundjoin
\pgfsetmiterlimit{7.0}
\definecolor{eps2pgf_color}{gray}{0}\pgfsetstrokecolor{eps2pgf_color}\pgfsetfillcolor{eps2pgf_color}
\pgfpathmoveto{\pgfqpoint{0.117cm}{1.815cm}}
\pgfpathlineto{\pgfqpoint{0.682cm}{1.065cm}}
\pgfpathlineto{\pgfqpoint{1.246cm}{1.815cm}}
\pgfusepath{stroke}
\end{pgfscope}
\definecolor{eps2pgf_color}{gray}{0}\pgfsetstrokecolor{eps2pgf_color}\pgfsetfillcolor{eps2pgf_color}
\pgfpathmoveto{\pgfqpoint{0.273cm}{1.789cm}}
\pgfpathcurveto{\pgfqpoint{0.273cm}{1.825cm}}{\pgfqpoint{0.259cm}{1.86cm}}{\pgfqpoint{0.233cm}{1.886cm}}
\pgfpathcurveto{\pgfqpoint{0.207cm}{1.912cm}}{\pgfqpoint{0.173cm}{1.926cm}}{\pgfqpoint{0.137cm}{1.926cm}}
\pgfpathcurveto{\pgfqpoint{0.1cm}{1.926cm}}{\pgfqpoint{0.066cm}{1.912cm}}{\pgfqpoint{0.04cm}{1.886cm}}
\pgfpathcurveto{\pgfqpoint{0.014cm}{1.86cm}}{\pgfqpoint{0cm}{1.825cm}}{\pgfqpoint{0cm}{1.789cm}}
\pgfpathcurveto{\pgfqpoint{0cm}{1.753cm}}{\pgfqpoint{0.014cm}{1.718cm}}{\pgfqpoint{0.04cm}{1.692cm}}
\pgfpathcurveto{\pgfqpoint{0.066cm}{1.667cm}}{\pgfqpoint{0.1cm}{1.652cm}}{\pgfqpoint{0.137cm}{1.652cm}}
\pgfpathcurveto{\pgfqpoint{0.173cm}{1.652cm}}{\pgfqpoint{0.207cm}{1.667cm}}{\pgfqpoint{0.233cm}{1.692cm}}
\pgfpathcurveto{\pgfqpoint{0.259cm}{1.718cm}}{\pgfqpoint{0.273cm}{1.753cm}}{\pgfqpoint{0.273cm}{1.789cm}}
\pgfusepath{fill}
\begin{pgfscope}
\pgfsetdash{}{0cm}
\pgfsetlinewidth{0.818mm}
\pgfsetmiterlimit{7.0}
\pgfpathmoveto{\pgfqpoint{0.682cm}{1.065cm}}
\pgfpathlineto{\pgfqpoint{0.679cm}{1.812cm}}
\pgfusepath{stroke}
\end{pgfscope}
\pgfpathmoveto{\pgfqpoint{0.815cm}{1.793cm}}
\pgfpathcurveto{\pgfqpoint{0.815cm}{1.829cm}}{\pgfqpoint{0.801cm}{1.864cm}}{\pgfqpoint{0.775cm}{1.89cm}}
\pgfpathcurveto{\pgfqpoint{0.75cm}{1.915cm}}{\pgfqpoint{0.715cm}{1.93cm}}{\pgfqpoint{0.679cm}{1.93cm}}
\pgfpathcurveto{\pgfqpoint{0.643cm}{1.93cm}}{\pgfqpoint{0.608cm}{1.915cm}}{\pgfqpoint{0.582cm}{1.89cm}}
\pgfpathcurveto{\pgfqpoint{0.557cm}{1.864cm}}{\pgfqpoint{0.542cm}{1.829cm}}{\pgfqpoint{0.542cm}{1.793cm}}
\pgfpathcurveto{\pgfqpoint{0.542cm}{1.756cm}}{\pgfqpoint{0.557cm}{1.722cm}}{\pgfqpoint{0.582cm}{1.696cm}}
\pgfpathcurveto{\pgfqpoint{0.608cm}{1.67cm}}{\pgfqpoint{0.643cm}{1.656cm}}{\pgfqpoint{0.679cm}{1.656cm}}
\pgfpathcurveto{\pgfqpoint{0.715cm}{1.656cm}}{\pgfqpoint{0.75cm}{1.67cm}}{\pgfqpoint{0.775cm}{1.696cm}}
\pgfpathcurveto{\pgfqpoint{0.801cm}{1.722cm}}{\pgfqpoint{0.815cm}{1.756cm}}{\pgfqpoint{0.815cm}{1.793cm}}
\pgfusepath{fill}
\pgfpathmoveto{\pgfqpoint{1.345cm}{1.765cm}}
\pgfpathcurveto{\pgfqpoint{1.345cm}{1.801cm}}{\pgfqpoint{1.331cm}{1.836cm}}{\pgfqpoint{1.305cm}{1.862cm}}
\pgfpathcurveto{\pgfqpoint{1.28cm}{1.887cm}}{\pgfqpoint{1.245cm}{1.902cm}}{\pgfqpoint{1.209cm}{1.902cm}}
\pgfpathcurveto{\pgfqpoint{1.172cm}{1.902cm}}{\pgfqpoint{1.138cm}{1.887cm}}{\pgfqpoint{1.112cm}{1.862cm}}
\pgfpathcurveto{\pgfqpoint{1.087cm}{1.836cm}}{\pgfqpoint{1.072cm}{1.801cm}}{\pgfqpoint{1.072cm}{1.765cm}}
\pgfpathcurveto{\pgfqpoint{1.072cm}{1.728cm}}{\pgfqpoint{1.087cm}{1.694cm}}{\pgfqpoint{1.112cm}{1.668cm}}
\pgfpathcurveto{\pgfqpoint{1.138cm}{1.642cm}}{\pgfqpoint{1.172cm}{1.628cm}}{\pgfqpoint{1.209cm}{1.628cm}}
\pgfpathcurveto{\pgfqpoint{1.245cm}{1.628cm}}{\pgfqpoint{1.28cm}{1.642cm}}{\pgfqpoint{1.305cm}{1.668cm}}
\pgfpathcurveto{\pgfqpoint{1.331cm}{1.694cm}}{\pgfqpoint{1.345cm}{1.728cm}}{\pgfqpoint{1.345cm}{1.765cm}}
\pgfusepath{fill}
\begin{pgfscope}
\pgfsetdash{}{0cm}
\pgfsetlinewidth{0.818mm}
\pgfsetroundcap
\pgfsetroundjoin
\pgfsetmiterlimit{7.0}
\pgfpathmoveto{\pgfqpoint{0.682cm}{1.065cm}}
\pgfpathlineto{\pgfqpoint{1.246cm}{0.315cm}}
\pgfpathlineto{\pgfqpoint{1.811cm}{1.065cm}}
\pgfusepath{stroke}
\end{pgfscope}
\pgfpathmoveto{\pgfqpoint{1.948cm}{1.065cm}}
\pgfpathcurveto{\pgfqpoint{1.948cm}{1.101cm}}{\pgfqpoint{1.933cm}{1.136cm}}{\pgfqpoint{1.907cm}{1.162cm}}
\pgfpathcurveto{\pgfqpoint{1.882cm}{1.187cm}}{\pgfqpoint{1.847cm}{1.202cm}}{\pgfqpoint{1.811cm}{1.202cm}}
\pgfpathcurveto{\pgfqpoint{1.775cm}{1.202cm}}{\pgfqpoint{1.74cm}{1.187cm}}{\pgfqpoint{1.714cm}{1.162cm}}
\pgfpathcurveto{\pgfqpoint{1.689cm}{1.136cm}}{\pgfqpoint{1.674cm}{1.101cm}}{\pgfqpoint{1.674cm}{1.065cm}}
\pgfpathcurveto{\pgfqpoint{1.674cm}{1.029cm}}{\pgfqpoint{1.689cm}{0.994cm}}{\pgfqpoint{1.714cm}{0.968cm}}
\pgfpathcurveto{\pgfqpoint{1.74cm}{0.942cm}}{\pgfqpoint{1.775cm}{0.928cm}}{\pgfqpoint{1.811cm}{0.928cm}}
\pgfpathcurveto{\pgfqpoint{1.847cm}{0.928cm}}{\pgfqpoint{1.882cm}{0.942cm}}{\pgfqpoint{1.907cm}{0.968cm}}
\pgfpathcurveto{\pgfqpoint{1.933cm}{0.994cm}}{\pgfqpoint{1.948cm}{1.029cm}}{\pgfqpoint{1.948cm}{1.065cm}}
\pgfusepath{fill}
\begin{pgfscope}
\pgfsetdash{}{0cm}
\pgfsetlinewidth{0.818mm}
\pgfsetmiterlimit{4.0}
\pgfpathmoveto{\pgfqpoint{1.383cm}{0.178cm}}
\pgfpathcurveto{\pgfqpoint{1.383cm}{0.214cm}}{\pgfqpoint{1.369cm}{0.249cm}}{\pgfqpoint{1.343cm}{0.275cm}}
\pgfpathcurveto{\pgfqpoint{1.317cm}{0.3cm}}{\pgfqpoint{1.283cm}{0.315cm}}{\pgfqpoint{1.246cm}{0.315cm}}
\pgfpathcurveto{\pgfqpoint{1.21cm}{0.315cm}}{\pgfqpoint{1.175cm}{0.3cm}}{\pgfqpoint{1.15cm}{0.275cm}}
\pgfpathcurveto{\pgfqpoint{1.124cm}{0.249cm}}{\pgfqpoint{1.11cm}{0.214cm}}{\pgfqpoint{1.11cm}{0.178cm}}
\pgfpathcurveto{\pgfqpoint{1.11cm}{0.141cm}}{\pgfqpoint{1.124cm}{0.107cm}}{\pgfqpoint{1.15cm}{0.081cm}}
\pgfpathcurveto{\pgfqpoint{1.175cm}{0.055cm}}{\pgfqpoint{1.21cm}{0.041cm}}{\pgfqpoint{1.246cm}{0.041cm}}
\pgfpathcurveto{\pgfqpoint{1.283cm}{0.041cm}}{\pgfqpoint{1.317cm}{0.055cm}}{\pgfqpoint{1.343cm}{0.081cm}}
\pgfpathcurveto{\pgfqpoint{1.369cm}{0.107cm}}{\pgfqpoint{1.383cm}{0.141cm}}{\pgfqpoint{1.383cm}{0.178cm}}
\pgfusepath{stroke}
\end{pgfscope}
\end{pgfscope}
\end{pgfscope}
\end{pgfscope}
\end{tikzpicture}}}}\\
&\qquad+\rmm{3\UU_> X\succ(\phi+\psi)^2},
\end{align*}
\begin{align*}
\rmb{\Psi}&:=\rmb{3\llbracket X^2 \rrbracket\circ\psi}+\rmb{3\llbracket X^2 \rrbracket\circ\vartheta}-\rmb{9\mathrm{com}(-X^{\!\resizebox{0.6em}{!}{
\begin{tikzpicture}
\pgfpathmoveto{\pgfqpoint{0cm}{-0.035cm}}
\pgfpathlineto{\pgfqpoint{1.376cm}{-0.035cm}}
\pgfpathlineto{\pgfqpoint{1.376cm}{1.552cm}}
\pgfpathlineto{\pgfqpoint{0cm}{1.552cm}}
\pgfpathclose
\pgfusepath{clip}
\begin{pgfscope}
\begin{pgfscope}
\pgfpathmoveto{\pgfqpoint{0cm}{-0.035cm}}
\pgfpathlineto{\pgfqpoint{1.376cm}{-0.035cm}}
\pgfpathlineto{\pgfqpoint{1.376cm}{1.552cm}}
\pgfpathlineto{\pgfqpoint{0cm}{1.552cm}}
\pgfpathclose
\pgfusepath{clip}
\begin{pgfscope}
\begin{pgfscope}
\pgfsetdash{}{0cm}
\pgfsetlinewidth{0.818mm}
\pgfsetroundcap
\pgfsetroundjoin
\pgfsetmiterlimit{7.0}
\definecolor{eps2pgf_color}{gray}{0}\pgfsetstrokecolor{eps2pgf_color}\pgfsetfillcolor{eps2pgf_color}
\pgfpathmoveto{\pgfqpoint{0.117cm}{1.421cm}}
\pgfpathlineto{\pgfqpoint{0.682cm}{0.671cm}}
\pgfpathlineto{\pgfqpoint{1.246cm}{1.421cm}}
\pgfusepath{stroke}
\end{pgfscope}
\definecolor{eps2pgf_color}{gray}{0}\pgfsetstrokecolor{eps2pgf_color}\pgfsetfillcolor{eps2pgf_color}
\pgfpathmoveto{\pgfqpoint{0.273cm}{1.395cm}}
\pgfpathcurveto{\pgfqpoint{0.273cm}{1.432cm}}{\pgfqpoint{0.259cm}{1.467cm}}{\pgfqpoint{0.233cm}{1.492cm}}
\pgfpathcurveto{\pgfqpoint{0.207cm}{1.518cm}}{\pgfqpoint{0.173cm}{1.532cm}}{\pgfqpoint{0.137cm}{1.532cm}}
\pgfpathcurveto{\pgfqpoint{0.1cm}{1.532cm}}{\pgfqpoint{0.066cm}{1.518cm}}{\pgfqpoint{0.04cm}{1.492cm}}
\pgfpathcurveto{\pgfqpoint{0.014cm}{1.467cm}}{\pgfqpoint{0cm}{1.432cm}}{\pgfqpoint{0cm}{1.395cm}}
\pgfpathcurveto{\pgfqpoint{0cm}{1.359cm}}{\pgfqpoint{0.014cm}{1.324cm}}{\pgfqpoint{0.04cm}{1.299cm}}
\pgfpathcurveto{\pgfqpoint{0.066cm}{1.273cm}}{\pgfqpoint{0.1cm}{1.258cm}}{\pgfqpoint{0.137cm}{1.258cm}}
\pgfpathcurveto{\pgfqpoint{0.173cm}{1.258cm}}{\pgfqpoint{0.207cm}{1.273cm}}{\pgfqpoint{0.233cm}{1.299cm}}
\pgfpathcurveto{\pgfqpoint{0.259cm}{1.324cm}}{\pgfqpoint{0.273cm}{1.359cm}}{\pgfqpoint{0.273cm}{1.395cm}}
\pgfusepath{fill}
\begin{pgfscope}
\pgfsetdash{}{0cm}
\pgfsetlinewidth{0.818mm}
\pgfsetmiterlimit{7.0}
\pgfpathmoveto{\pgfqpoint{0.682cm}{0.671cm}}
\pgfpathlineto{\pgfqpoint{0.679cm}{1.418cm}}
\pgfusepath{stroke}
\end{pgfscope}
\pgfpathmoveto{\pgfqpoint{0.815cm}{1.399cm}}
\pgfpathcurveto{\pgfqpoint{0.815cm}{1.435cm}}{\pgfqpoint{0.801cm}{1.47cm}}{\pgfqpoint{0.775cm}{1.496cm}}
\pgfpathcurveto{\pgfqpoint{0.75cm}{1.521cm}}{\pgfqpoint{0.715cm}{1.536cm}}{\pgfqpoint{0.679cm}{1.536cm}}
\pgfpathcurveto{\pgfqpoint{0.643cm}{1.536cm}}{\pgfqpoint{0.608cm}{1.521cm}}{\pgfqpoint{0.582cm}{1.496cm}}
\pgfpathcurveto{\pgfqpoint{0.557cm}{1.47cm}}{\pgfqpoint{0.542cm}{1.435cm}}{\pgfqpoint{0.542cm}{1.399cm}}
\pgfpathcurveto{\pgfqpoint{0.542cm}{1.363cm}}{\pgfqpoint{0.557cm}{1.328cm}}{\pgfqpoint{0.582cm}{1.302cm}}
\pgfpathcurveto{\pgfqpoint{0.608cm}{1.276cm}}{\pgfqpoint{0.643cm}{1.262cm}}{\pgfqpoint{0.679cm}{1.262cm}}
\pgfpathcurveto{\pgfqpoint{0.715cm}{1.262cm}}{\pgfqpoint{0.75cm}{1.276cm}}{\pgfqpoint{0.775cm}{1.302cm}}
\pgfpathcurveto{\pgfqpoint{0.801cm}{1.328cm}}{\pgfqpoint{0.815cm}{1.363cm}}{\pgfqpoint{0.815cm}{1.399cm}}
\pgfusepath{fill}
\pgfpathmoveto{\pgfqpoint{1.345cm}{1.371cm}}
\pgfpathcurveto{\pgfqpoint{1.345cm}{1.408cm}}{\pgfqpoint{1.331cm}{1.442cm}}{\pgfqpoint{1.305cm}{1.468cm}}
\pgfpathcurveto{\pgfqpoint{1.28cm}{1.494cm}}{\pgfqpoint{1.245cm}{1.508cm}}{\pgfqpoint{1.209cm}{1.508cm}}
\pgfpathcurveto{\pgfqpoint{1.172cm}{1.508cm}}{\pgfqpoint{1.138cm}{1.494cm}}{\pgfqpoint{1.112cm}{1.468cm}}
\pgfpathcurveto{\pgfqpoint{1.087cm}{1.442cm}}{\pgfqpoint{1.072cm}{1.408cm}}{\pgfqpoint{1.072cm}{1.371cm}}
\pgfpathcurveto{\pgfqpoint{1.072cm}{1.335cm}}{\pgfqpoint{1.087cm}{1.3cm}}{\pgfqpoint{1.112cm}{1.274cm}}
\pgfpathcurveto{\pgfqpoint{1.138cm}{1.249cm}}{\pgfqpoint{1.172cm}{1.234cm}}{\pgfqpoint{1.209cm}{1.234cm}}
\pgfpathcurveto{\pgfqpoint{1.245cm}{1.234cm}}{\pgfqpoint{1.28cm}{1.249cm}}{\pgfqpoint{1.305cm}{1.274cm}}
\pgfpathcurveto{\pgfqpoint{1.331cm}{1.3cm}}{\pgfqpoint{1.345cm}{1.335cm}}{\pgfqpoint{1.345cm}{1.371cm}}
\pgfusepath{fill}
\begin{pgfscope}
\pgfsetdash{}{0cm}
\pgfsetlinewidth{0.818mm}
\pgfsetroundcap
\pgfsetmiterlimit{4.0}
\pgfpathmoveto{\pgfqpoint{0.682cm}{0.671cm}}
\pgfpathlineto{\pgfqpoint{0.682cm}{0.042cm}}
\pgfusepath{stroke}
\end{pgfscope}
\end{pgfscope}
\end{pgfscope}
\end{pgfscope}
\end{tikzpicture}}} + \phi + \psi,X^{\!\resizebox{0.6em}{!}{
\begin{tikzpicture}
\pgfpathmoveto{\pgfqpoint{0cm}{0cm}}
\pgfpathlineto{\pgfqpoint{1.376cm}{0cm}}
\pgfpathlineto{\pgfqpoint{1.376cm}{1.588cm}}
\pgfpathlineto{\pgfqpoint{0cm}{1.588cm}}
\pgfpathclose
\pgfusepath{clip}
\begin{pgfscope}
\begin{pgfscope}
\pgfpathmoveto{\pgfqpoint{0cm}{0cm}}
\pgfpathlineto{\pgfqpoint{1.376cm}{0cm}}
\pgfpathlineto{\pgfqpoint{1.376cm}{1.588cm}}
\pgfpathlineto{\pgfqpoint{0cm}{1.588cm}}
\pgfpathclose
\pgfusepath{clip}
\begin{pgfscope}
\begin{pgfscope}
\definecolor{eps2pgf_color}{gray}{0.976471}\pgfsetstrokecolor{eps2pgf_color}\pgfsetfillcolor{eps2pgf_color}
\pgfpathmoveto{\pgfqpoint{0cm}{0cm}}
\pgfpathlineto{\pgfqpoint{1.376cm}{0cm}}
\pgfpathlineto{\pgfqpoint{1.376cm}{1.588cm}}
\pgfpathlineto{\pgfqpoint{0cm}{1.588cm}}
\pgfpathclose
\pgfusepath{fill}
\end{pgfscope}
\begin{pgfscope}
\pgfsetdash{}{0cm}
\pgfsetlinewidth{0.818mm}
\pgfsetroundcap
\pgfsetroundjoin
\pgfsetmiterlimit{7.0}
\definecolor{eps2pgf_color}{gray}{0}\pgfsetstrokecolor{eps2pgf_color}\pgfsetfillcolor{eps2pgf_color}
\pgfpathmoveto{\pgfqpoint{0.117cm}{1.476cm}}
\pgfpathlineto{\pgfqpoint{0.682cm}{0.726cm}}
\pgfpathlineto{\pgfqpoint{1.246cm}{1.476cm}}
\pgfusepath{stroke}
\end{pgfscope}
\definecolor{eps2pgf_color}{gray}{0}\pgfsetstrokecolor{eps2pgf_color}\pgfsetfillcolor{eps2pgf_color}
\pgfpathmoveto{\pgfqpoint{0.273cm}{1.451cm}}
\pgfpathcurveto{\pgfqpoint{0.273cm}{1.487cm}}{\pgfqpoint{0.259cm}{1.522cm}}{\pgfqpoint{0.233cm}{1.547cm}}
\pgfpathcurveto{\pgfqpoint{0.207cm}{1.573cm}}{\pgfqpoint{0.173cm}{1.588cm}}{\pgfqpoint{0.137cm}{1.588cm}}
\pgfpathcurveto{\pgfqpoint{0.1cm}{1.588cm}}{\pgfqpoint{0.066cm}{1.573cm}}{\pgfqpoint{0.04cm}{1.547cm}}
\pgfpathcurveto{\pgfqpoint{0.014cm}{1.522cm}}{\pgfqpoint{0cm}{1.487cm}}{\pgfqpoint{0cm}{1.451cm}}
\pgfpathcurveto{\pgfqpoint{0cm}{1.414cm}}{\pgfqpoint{0.014cm}{1.379cm}}{\pgfqpoint{0.04cm}{1.354cm}}
\pgfpathcurveto{\pgfqpoint{0.066cm}{1.328cm}}{\pgfqpoint{0.1cm}{1.314cm}}{\pgfqpoint{0.137cm}{1.314cm}}
\pgfpathcurveto{\pgfqpoint{0.173cm}{1.314cm}}{\pgfqpoint{0.207cm}{1.328cm}}{\pgfqpoint{0.233cm}{1.354cm}}
\pgfpathcurveto{\pgfqpoint{0.259cm}{1.379cm}}{\pgfqpoint{0.273cm}{1.414cm}}{\pgfqpoint{0.273cm}{1.451cm}}
\pgfusepath{fill}
\pgfpathmoveto{\pgfqpoint{1.345cm}{1.426cm}}
\pgfpathcurveto{\pgfqpoint{1.345cm}{1.463cm}}{\pgfqpoint{1.331cm}{1.497cm}}{\pgfqpoint{1.305cm}{1.523cm}}
\pgfpathcurveto{\pgfqpoint{1.28cm}{1.549cm}}{\pgfqpoint{1.245cm}{1.563cm}}{\pgfqpoint{1.209cm}{1.563cm}}
\pgfpathcurveto{\pgfqpoint{1.172cm}{1.563cm}}{\pgfqpoint{1.138cm}{1.549cm}}{\pgfqpoint{1.112cm}{1.523cm}}
\pgfpathcurveto{\pgfqpoint{1.087cm}{1.497cm}}{\pgfqpoint{1.072cm}{1.463cm}}{\pgfqpoint{1.072cm}{1.426cm}}
\pgfpathcurveto{\pgfqpoint{1.072cm}{1.39cm}}{\pgfqpoint{1.087cm}{1.355cm}}{\pgfqpoint{1.112cm}{1.329cm}}
\pgfpathcurveto{\pgfqpoint{1.138cm}{1.304cm}}{\pgfqpoint{1.172cm}{1.289cm}}{\pgfqpoint{1.209cm}{1.289cm}}
\pgfpathcurveto{\pgfqpoint{1.245cm}{1.289cm}}{\pgfqpoint{1.28cm}{1.304cm}}{\pgfqpoint{1.305cm}{1.329cm}}
\pgfpathcurveto{\pgfqpoint{1.331cm}{1.355cm}}{\pgfqpoint{1.345cm}{1.39cm}}{\pgfqpoint{1.345cm}{1.426cm}}
\pgfusepath{fill}
\begin{pgfscope}
\pgfsetdash{}{0cm}
\pgfsetlinewidth{0.818mm}
\pgfsetroundcap
\pgfsetmiterlimit{4.0}
\pgfpathmoveto{\pgfqpoint{0.682cm}{0.726cm}}
\pgfpathlineto{\pgfqpoint{0.682cm}{0.097cm}}
\pgfusepath{stroke}
\end{pgfscope}
\end{pgfscope}
\end{pgfscope}
\end{pgfscope}
\end{tikzpicture}}},\llbracket X^2 \rrbracket)}+\rmb{6\mathrm{com}(X^{\!\resizebox{0.6em}{!}{
\begin{tikzpicture}
\pgfpathmoveto{\pgfqpoint{0cm}{-0.035cm}}
\pgfpathlineto{\pgfqpoint{1.376cm}{-0.035cm}}
\pgfpathlineto{\pgfqpoint{1.376cm}{1.552cm}}
\pgfpathlineto{\pgfqpoint{0cm}{1.552cm}}
\pgfpathclose
\pgfusepath{clip}
\begin{pgfscope}
\begin{pgfscope}
\pgfpathmoveto{\pgfqpoint{0cm}{-0.035cm}}
\pgfpathlineto{\pgfqpoint{1.376cm}{-0.035cm}}
\pgfpathlineto{\pgfqpoint{1.376cm}{1.552cm}}
\pgfpathlineto{\pgfqpoint{0cm}{1.552cm}}
\pgfpathclose
\pgfusepath{clip}
\begin{pgfscope}
\begin{pgfscope}
\pgfsetdash{}{0cm}
\pgfsetlinewidth{0.818mm}
\pgfsetroundcap
\pgfsetroundjoin
\pgfsetmiterlimit{7.0}
\definecolor{eps2pgf_color}{gray}{0}\pgfsetstrokecolor{eps2pgf_color}\pgfsetfillcolor{eps2pgf_color}
\pgfpathmoveto{\pgfqpoint{0.117cm}{1.421cm}}
\pgfpathlineto{\pgfqpoint{0.682cm}{0.671cm}}
\pgfpathlineto{\pgfqpoint{1.246cm}{1.421cm}}
\pgfusepath{stroke}
\end{pgfscope}
\definecolor{eps2pgf_color}{gray}{0}\pgfsetstrokecolor{eps2pgf_color}\pgfsetfillcolor{eps2pgf_color}
\pgfpathmoveto{\pgfqpoint{0.273cm}{1.395cm}}
\pgfpathcurveto{\pgfqpoint{0.273cm}{1.432cm}}{\pgfqpoint{0.259cm}{1.467cm}}{\pgfqpoint{0.233cm}{1.492cm}}
\pgfpathcurveto{\pgfqpoint{0.207cm}{1.518cm}}{\pgfqpoint{0.173cm}{1.532cm}}{\pgfqpoint{0.137cm}{1.532cm}}
\pgfpathcurveto{\pgfqpoint{0.1cm}{1.532cm}}{\pgfqpoint{0.066cm}{1.518cm}}{\pgfqpoint{0.04cm}{1.492cm}}
\pgfpathcurveto{\pgfqpoint{0.014cm}{1.467cm}}{\pgfqpoint{0cm}{1.432cm}}{\pgfqpoint{0cm}{1.395cm}}
\pgfpathcurveto{\pgfqpoint{0cm}{1.359cm}}{\pgfqpoint{0.014cm}{1.324cm}}{\pgfqpoint{0.04cm}{1.299cm}}
\pgfpathcurveto{\pgfqpoint{0.066cm}{1.273cm}}{\pgfqpoint{0.1cm}{1.258cm}}{\pgfqpoint{0.137cm}{1.258cm}}
\pgfpathcurveto{\pgfqpoint{0.173cm}{1.258cm}}{\pgfqpoint{0.207cm}{1.273cm}}{\pgfqpoint{0.233cm}{1.299cm}}
\pgfpathcurveto{\pgfqpoint{0.259cm}{1.324cm}}{\pgfqpoint{0.273cm}{1.359cm}}{\pgfqpoint{0.273cm}{1.395cm}}
\pgfusepath{fill}
\begin{pgfscope}
\pgfsetdash{}{0cm}
\pgfsetlinewidth{0.818mm}
\pgfsetmiterlimit{7.0}
\pgfpathmoveto{\pgfqpoint{0.682cm}{0.671cm}}
\pgfpathlineto{\pgfqpoint{0.679cm}{1.418cm}}
\pgfusepath{stroke}
\end{pgfscope}
\pgfpathmoveto{\pgfqpoint{0.815cm}{1.399cm}}
\pgfpathcurveto{\pgfqpoint{0.815cm}{1.435cm}}{\pgfqpoint{0.801cm}{1.47cm}}{\pgfqpoint{0.775cm}{1.496cm}}
\pgfpathcurveto{\pgfqpoint{0.75cm}{1.521cm}}{\pgfqpoint{0.715cm}{1.536cm}}{\pgfqpoint{0.679cm}{1.536cm}}
\pgfpathcurveto{\pgfqpoint{0.643cm}{1.536cm}}{\pgfqpoint{0.608cm}{1.521cm}}{\pgfqpoint{0.582cm}{1.496cm}}
\pgfpathcurveto{\pgfqpoint{0.557cm}{1.47cm}}{\pgfqpoint{0.542cm}{1.435cm}}{\pgfqpoint{0.542cm}{1.399cm}}
\pgfpathcurveto{\pgfqpoint{0.542cm}{1.363cm}}{\pgfqpoint{0.557cm}{1.328cm}}{\pgfqpoint{0.582cm}{1.302cm}}
\pgfpathcurveto{\pgfqpoint{0.608cm}{1.276cm}}{\pgfqpoint{0.643cm}{1.262cm}}{\pgfqpoint{0.679cm}{1.262cm}}
\pgfpathcurveto{\pgfqpoint{0.715cm}{1.262cm}}{\pgfqpoint{0.75cm}{1.276cm}}{\pgfqpoint{0.775cm}{1.302cm}}
\pgfpathcurveto{\pgfqpoint{0.801cm}{1.328cm}}{\pgfqpoint{0.815cm}{1.363cm}}{\pgfqpoint{0.815cm}{1.399cm}}
\pgfusepath{fill}
\pgfpathmoveto{\pgfqpoint{1.345cm}{1.371cm}}
\pgfpathcurveto{\pgfqpoint{1.345cm}{1.408cm}}{\pgfqpoint{1.331cm}{1.442cm}}{\pgfqpoint{1.305cm}{1.468cm}}
\pgfpathcurveto{\pgfqpoint{1.28cm}{1.494cm}}{\pgfqpoint{1.245cm}{1.508cm}}{\pgfqpoint{1.209cm}{1.508cm}}
\pgfpathcurveto{\pgfqpoint{1.172cm}{1.508cm}}{\pgfqpoint{1.138cm}{1.494cm}}{\pgfqpoint{1.112cm}{1.468cm}}
\pgfpathcurveto{\pgfqpoint{1.087cm}{1.442cm}}{\pgfqpoint{1.072cm}{1.408cm}}{\pgfqpoint{1.072cm}{1.371cm}}
\pgfpathcurveto{\pgfqpoint{1.072cm}{1.335cm}}{\pgfqpoint{1.087cm}{1.3cm}}{\pgfqpoint{1.112cm}{1.274cm}}
\pgfpathcurveto{\pgfqpoint{1.138cm}{1.249cm}}{\pgfqpoint{1.172cm}{1.234cm}}{\pgfqpoint{1.209cm}{1.234cm}}
\pgfpathcurveto{\pgfqpoint{1.245cm}{1.234cm}}{\pgfqpoint{1.28cm}{1.249cm}}{\pgfqpoint{1.305cm}{1.274cm}}
\pgfpathcurveto{\pgfqpoint{1.331cm}{1.3cm}}{\pgfqpoint{1.345cm}{1.335cm}}{\pgfqpoint{1.345cm}{1.371cm}}
\pgfusepath{fill}
\begin{pgfscope}
\pgfsetdash{}{0cm}
\pgfsetlinewidth{0.818mm}
\pgfsetroundcap
\pgfsetmiterlimit{4.0}
\pgfpathmoveto{\pgfqpoint{0.682cm}{0.671cm}}
\pgfpathlineto{\pgfqpoint{0.682cm}{0.042cm}}
\pgfusepath{stroke}
\end{pgfscope}
\end{pgfscope}
\end{pgfscope}
\end{pgfscope}
\end{tikzpicture}}},X^{\!\resizebox{0.6em}{!}{
\begin{tikzpicture}
\pgfpathmoveto{\pgfqpoint{0cm}{-0.035cm}}
\pgfpathlineto{\pgfqpoint{1.376cm}{-0.035cm}}
\pgfpathlineto{\pgfqpoint{1.376cm}{1.552cm}}
\pgfpathlineto{\pgfqpoint{0cm}{1.552cm}}
\pgfpathclose
\pgfusepath{clip}
\begin{pgfscope}
\begin{pgfscope}
\pgfpathmoveto{\pgfqpoint{0cm}{-0.035cm}}
\pgfpathlineto{\pgfqpoint{1.376cm}{-0.035cm}}
\pgfpathlineto{\pgfqpoint{1.376cm}{1.552cm}}
\pgfpathlineto{\pgfqpoint{0cm}{1.552cm}}
\pgfpathclose
\pgfusepath{clip}
\begin{pgfscope}
\begin{pgfscope}
\pgfsetdash{}{0cm}
\pgfsetlinewidth{0.818mm}
\pgfsetroundcap
\pgfsetroundjoin
\pgfsetmiterlimit{7.0}
\definecolor{eps2pgf_color}{gray}{0}\pgfsetstrokecolor{eps2pgf_color}\pgfsetfillcolor{eps2pgf_color}
\pgfpathmoveto{\pgfqpoint{0.117cm}{1.421cm}}
\pgfpathlineto{\pgfqpoint{0.682cm}{0.671cm}}
\pgfpathlineto{\pgfqpoint{1.246cm}{1.421cm}}
\pgfusepath{stroke}
\end{pgfscope}
\definecolor{eps2pgf_color}{gray}{0}\pgfsetstrokecolor{eps2pgf_color}\pgfsetfillcolor{eps2pgf_color}
\pgfpathmoveto{\pgfqpoint{0.273cm}{1.395cm}}
\pgfpathcurveto{\pgfqpoint{0.273cm}{1.432cm}}{\pgfqpoint{0.259cm}{1.467cm}}{\pgfqpoint{0.233cm}{1.492cm}}
\pgfpathcurveto{\pgfqpoint{0.207cm}{1.518cm}}{\pgfqpoint{0.173cm}{1.532cm}}{\pgfqpoint{0.137cm}{1.532cm}}
\pgfpathcurveto{\pgfqpoint{0.1cm}{1.532cm}}{\pgfqpoint{0.066cm}{1.518cm}}{\pgfqpoint{0.04cm}{1.492cm}}
\pgfpathcurveto{\pgfqpoint{0.014cm}{1.467cm}}{\pgfqpoint{0cm}{1.432cm}}{\pgfqpoint{0cm}{1.395cm}}
\pgfpathcurveto{\pgfqpoint{0cm}{1.359cm}}{\pgfqpoint{0.014cm}{1.324cm}}{\pgfqpoint{0.04cm}{1.299cm}}
\pgfpathcurveto{\pgfqpoint{0.066cm}{1.273cm}}{\pgfqpoint{0.1cm}{1.258cm}}{\pgfqpoint{0.137cm}{1.258cm}}
\pgfpathcurveto{\pgfqpoint{0.173cm}{1.258cm}}{\pgfqpoint{0.207cm}{1.273cm}}{\pgfqpoint{0.233cm}{1.299cm}}
\pgfpathcurveto{\pgfqpoint{0.259cm}{1.324cm}}{\pgfqpoint{0.273cm}{1.359cm}}{\pgfqpoint{0.273cm}{1.395cm}}
\pgfusepath{fill}
\begin{pgfscope}
\pgfsetdash{}{0cm}
\pgfsetlinewidth{0.818mm}
\pgfsetmiterlimit{7.0}
\pgfpathmoveto{\pgfqpoint{0.682cm}{0.671cm}}
\pgfpathlineto{\pgfqpoint{0.679cm}{1.418cm}}
\pgfusepath{stroke}
\end{pgfscope}
\pgfpathmoveto{\pgfqpoint{0.815cm}{1.399cm}}
\pgfpathcurveto{\pgfqpoint{0.815cm}{1.435cm}}{\pgfqpoint{0.801cm}{1.47cm}}{\pgfqpoint{0.775cm}{1.496cm}}
\pgfpathcurveto{\pgfqpoint{0.75cm}{1.521cm}}{\pgfqpoint{0.715cm}{1.536cm}}{\pgfqpoint{0.679cm}{1.536cm}}
\pgfpathcurveto{\pgfqpoint{0.643cm}{1.536cm}}{\pgfqpoint{0.608cm}{1.521cm}}{\pgfqpoint{0.582cm}{1.496cm}}
\pgfpathcurveto{\pgfqpoint{0.557cm}{1.47cm}}{\pgfqpoint{0.542cm}{1.435cm}}{\pgfqpoint{0.542cm}{1.399cm}}
\pgfpathcurveto{\pgfqpoint{0.542cm}{1.363cm}}{\pgfqpoint{0.557cm}{1.328cm}}{\pgfqpoint{0.582cm}{1.302cm}}
\pgfpathcurveto{\pgfqpoint{0.608cm}{1.276cm}}{\pgfqpoint{0.643cm}{1.262cm}}{\pgfqpoint{0.679cm}{1.262cm}}
\pgfpathcurveto{\pgfqpoint{0.715cm}{1.262cm}}{\pgfqpoint{0.75cm}{1.276cm}}{\pgfqpoint{0.775cm}{1.302cm}}
\pgfpathcurveto{\pgfqpoint{0.801cm}{1.328cm}}{\pgfqpoint{0.815cm}{1.363cm}}{\pgfqpoint{0.815cm}{1.399cm}}
\pgfusepath{fill}
\pgfpathmoveto{\pgfqpoint{1.345cm}{1.371cm}}
\pgfpathcurveto{\pgfqpoint{1.345cm}{1.408cm}}{\pgfqpoint{1.331cm}{1.442cm}}{\pgfqpoint{1.305cm}{1.468cm}}
\pgfpathcurveto{\pgfqpoint{1.28cm}{1.494cm}}{\pgfqpoint{1.245cm}{1.508cm}}{\pgfqpoint{1.209cm}{1.508cm}}
\pgfpathcurveto{\pgfqpoint{1.172cm}{1.508cm}}{\pgfqpoint{1.138cm}{1.494cm}}{\pgfqpoint{1.112cm}{1.468cm}}
\pgfpathcurveto{\pgfqpoint{1.087cm}{1.442cm}}{\pgfqpoint{1.072cm}{1.408cm}}{\pgfqpoint{1.072cm}{1.371cm}}
\pgfpathcurveto{\pgfqpoint{1.072cm}{1.335cm}}{\pgfqpoint{1.087cm}{1.3cm}}{\pgfqpoint{1.112cm}{1.274cm}}
\pgfpathcurveto{\pgfqpoint{1.138cm}{1.249cm}}{\pgfqpoint{1.172cm}{1.234cm}}{\pgfqpoint{1.209cm}{1.234cm}}
\pgfpathcurveto{\pgfqpoint{1.245cm}{1.234cm}}{\pgfqpoint{1.28cm}{1.249cm}}{\pgfqpoint{1.305cm}{1.274cm}}
\pgfpathcurveto{\pgfqpoint{1.331cm}{1.3cm}}{\pgfqpoint{1.345cm}{1.335cm}}{\pgfqpoint{1.345cm}{1.371cm}}
\pgfusepath{fill}
\begin{pgfscope}
\pgfsetdash{}{0cm}
\pgfsetlinewidth{0.818mm}
\pgfsetroundcap
\pgfsetmiterlimit{4.0}
\pgfpathmoveto{\pgfqpoint{0.682cm}{0.671cm}}
\pgfpathlineto{\pgfqpoint{0.682cm}{0.042cm}}
\pgfusepath{stroke}
\end{pgfscope}
\end{pgfscope}
\end{pgfscope}
\end{pgfscope}
\end{tikzpicture}}},X)}\\
&\qquad+\rmb{3X\circ (X^{\!\resizebox{0.6em}{!}{
\begin{tikzpicture}
\pgfpathmoveto{\pgfqpoint{0cm}{-0.035cm}}
\pgfpathlineto{\pgfqpoint{1.376cm}{-0.035cm}}
\pgfpathlineto{\pgfqpoint{1.376cm}{1.552cm}}
\pgfpathlineto{\pgfqpoint{0cm}{1.552cm}}
\pgfpathclose
\pgfusepath{clip}
\begin{pgfscope}
\begin{pgfscope}
\pgfpathmoveto{\pgfqpoint{0cm}{-0.035cm}}
\pgfpathlineto{\pgfqpoint{1.376cm}{-0.035cm}}
\pgfpathlineto{\pgfqpoint{1.376cm}{1.552cm}}
\pgfpathlineto{\pgfqpoint{0cm}{1.552cm}}
\pgfpathclose
\pgfusepath{clip}
\begin{pgfscope}
\begin{pgfscope}
\pgfsetdash{}{0cm}
\pgfsetlinewidth{0.818mm}
\pgfsetroundcap
\pgfsetroundjoin
\pgfsetmiterlimit{7.0}
\definecolor{eps2pgf_color}{gray}{0}\pgfsetstrokecolor{eps2pgf_color}\pgfsetfillcolor{eps2pgf_color}
\pgfpathmoveto{\pgfqpoint{0.117cm}{1.421cm}}
\pgfpathlineto{\pgfqpoint{0.682cm}{0.671cm}}
\pgfpathlineto{\pgfqpoint{1.246cm}{1.421cm}}
\pgfusepath{stroke}
\end{pgfscope}
\definecolor{eps2pgf_color}{gray}{0}\pgfsetstrokecolor{eps2pgf_color}\pgfsetfillcolor{eps2pgf_color}
\pgfpathmoveto{\pgfqpoint{0.273cm}{1.395cm}}
\pgfpathcurveto{\pgfqpoint{0.273cm}{1.432cm}}{\pgfqpoint{0.259cm}{1.467cm}}{\pgfqpoint{0.233cm}{1.492cm}}
\pgfpathcurveto{\pgfqpoint{0.207cm}{1.518cm}}{\pgfqpoint{0.173cm}{1.532cm}}{\pgfqpoint{0.137cm}{1.532cm}}
\pgfpathcurveto{\pgfqpoint{0.1cm}{1.532cm}}{\pgfqpoint{0.066cm}{1.518cm}}{\pgfqpoint{0.04cm}{1.492cm}}
\pgfpathcurveto{\pgfqpoint{0.014cm}{1.467cm}}{\pgfqpoint{0cm}{1.432cm}}{\pgfqpoint{0cm}{1.395cm}}
\pgfpathcurveto{\pgfqpoint{0cm}{1.359cm}}{\pgfqpoint{0.014cm}{1.324cm}}{\pgfqpoint{0.04cm}{1.299cm}}
\pgfpathcurveto{\pgfqpoint{0.066cm}{1.273cm}}{\pgfqpoint{0.1cm}{1.258cm}}{\pgfqpoint{0.137cm}{1.258cm}}
\pgfpathcurveto{\pgfqpoint{0.173cm}{1.258cm}}{\pgfqpoint{0.207cm}{1.273cm}}{\pgfqpoint{0.233cm}{1.299cm}}
\pgfpathcurveto{\pgfqpoint{0.259cm}{1.324cm}}{\pgfqpoint{0.273cm}{1.359cm}}{\pgfqpoint{0.273cm}{1.395cm}}
\pgfusepath{fill}
\begin{pgfscope}
\pgfsetdash{}{0cm}
\pgfsetlinewidth{0.818mm}
\pgfsetmiterlimit{7.0}
\pgfpathmoveto{\pgfqpoint{0.682cm}{0.671cm}}
\pgfpathlineto{\pgfqpoint{0.679cm}{1.418cm}}
\pgfusepath{stroke}
\end{pgfscope}
\pgfpathmoveto{\pgfqpoint{0.815cm}{1.399cm}}
\pgfpathcurveto{\pgfqpoint{0.815cm}{1.435cm}}{\pgfqpoint{0.801cm}{1.47cm}}{\pgfqpoint{0.775cm}{1.496cm}}
\pgfpathcurveto{\pgfqpoint{0.75cm}{1.521cm}}{\pgfqpoint{0.715cm}{1.536cm}}{\pgfqpoint{0.679cm}{1.536cm}}
\pgfpathcurveto{\pgfqpoint{0.643cm}{1.536cm}}{\pgfqpoint{0.608cm}{1.521cm}}{\pgfqpoint{0.582cm}{1.496cm}}
\pgfpathcurveto{\pgfqpoint{0.557cm}{1.47cm}}{\pgfqpoint{0.542cm}{1.435cm}}{\pgfqpoint{0.542cm}{1.399cm}}
\pgfpathcurveto{\pgfqpoint{0.542cm}{1.363cm}}{\pgfqpoint{0.557cm}{1.328cm}}{\pgfqpoint{0.582cm}{1.302cm}}
\pgfpathcurveto{\pgfqpoint{0.608cm}{1.276cm}}{\pgfqpoint{0.643cm}{1.262cm}}{\pgfqpoint{0.679cm}{1.262cm}}
\pgfpathcurveto{\pgfqpoint{0.715cm}{1.262cm}}{\pgfqpoint{0.75cm}{1.276cm}}{\pgfqpoint{0.775cm}{1.302cm}}
\pgfpathcurveto{\pgfqpoint{0.801cm}{1.328cm}}{\pgfqpoint{0.815cm}{1.363cm}}{\pgfqpoint{0.815cm}{1.399cm}}
\pgfusepath{fill}
\pgfpathmoveto{\pgfqpoint{1.345cm}{1.371cm}}
\pgfpathcurveto{\pgfqpoint{1.345cm}{1.408cm}}{\pgfqpoint{1.331cm}{1.442cm}}{\pgfqpoint{1.305cm}{1.468cm}}
\pgfpathcurveto{\pgfqpoint{1.28cm}{1.494cm}}{\pgfqpoint{1.245cm}{1.508cm}}{\pgfqpoint{1.209cm}{1.508cm}}
\pgfpathcurveto{\pgfqpoint{1.172cm}{1.508cm}}{\pgfqpoint{1.138cm}{1.494cm}}{\pgfqpoint{1.112cm}{1.468cm}}
\pgfpathcurveto{\pgfqpoint{1.087cm}{1.442cm}}{\pgfqpoint{1.072cm}{1.408cm}}{\pgfqpoint{1.072cm}{1.371cm}}
\pgfpathcurveto{\pgfqpoint{1.072cm}{1.335cm}}{\pgfqpoint{1.087cm}{1.3cm}}{\pgfqpoint{1.112cm}{1.274cm}}
\pgfpathcurveto{\pgfqpoint{1.138cm}{1.249cm}}{\pgfqpoint{1.172cm}{1.234cm}}{\pgfqpoint{1.209cm}{1.234cm}}
\pgfpathcurveto{\pgfqpoint{1.245cm}{1.234cm}}{\pgfqpoint{1.28cm}{1.249cm}}{\pgfqpoint{1.305cm}{1.274cm}}
\pgfpathcurveto{\pgfqpoint{1.331cm}{1.3cm}}{\pgfqpoint{1.345cm}{1.335cm}}{\pgfqpoint{1.345cm}{1.371cm}}
\pgfusepath{fill}
\begin{pgfscope}
\pgfsetdash{}{0cm}
\pgfsetlinewidth{0.818mm}
\pgfsetroundcap
\pgfsetmiterlimit{4.0}
\pgfpathmoveto{\pgfqpoint{0.682cm}{0.671cm}}
\pgfpathlineto{\pgfqpoint{0.682cm}{0.042cm}}
\pgfusepath{stroke}
\end{pgfscope}
\end{pgfscope}
\end{pgfscope}
\end{pgfscope}
\end{tikzpicture}}}\circ X^{\!\resizebox{0.6em}{!}{
\begin{tikzpicture}
\pgfpathmoveto{\pgfqpoint{0cm}{-0.035cm}}
\pgfpathlineto{\pgfqpoint{1.376cm}{-0.035cm}}
\pgfpathlineto{\pgfqpoint{1.376cm}{1.552cm}}
\pgfpathlineto{\pgfqpoint{0cm}{1.552cm}}
\pgfpathclose
\pgfusepath{clip}
\begin{pgfscope}
\begin{pgfscope}
\pgfpathmoveto{\pgfqpoint{0cm}{-0.035cm}}
\pgfpathlineto{\pgfqpoint{1.376cm}{-0.035cm}}
\pgfpathlineto{\pgfqpoint{1.376cm}{1.552cm}}
\pgfpathlineto{\pgfqpoint{0cm}{1.552cm}}
\pgfpathclose
\pgfusepath{clip}
\begin{pgfscope}
\begin{pgfscope}
\pgfsetdash{}{0cm}
\pgfsetlinewidth{0.818mm}
\pgfsetroundcap
\pgfsetroundjoin
\pgfsetmiterlimit{7.0}
\definecolor{eps2pgf_color}{gray}{0}\pgfsetstrokecolor{eps2pgf_color}\pgfsetfillcolor{eps2pgf_color}
\pgfpathmoveto{\pgfqpoint{0.117cm}{1.421cm}}
\pgfpathlineto{\pgfqpoint{0.682cm}{0.671cm}}
\pgfpathlineto{\pgfqpoint{1.246cm}{1.421cm}}
\pgfusepath{stroke}
\end{pgfscope}
\definecolor{eps2pgf_color}{gray}{0}\pgfsetstrokecolor{eps2pgf_color}\pgfsetfillcolor{eps2pgf_color}
\pgfpathmoveto{\pgfqpoint{0.273cm}{1.395cm}}
\pgfpathcurveto{\pgfqpoint{0.273cm}{1.432cm}}{\pgfqpoint{0.259cm}{1.467cm}}{\pgfqpoint{0.233cm}{1.492cm}}
\pgfpathcurveto{\pgfqpoint{0.207cm}{1.518cm}}{\pgfqpoint{0.173cm}{1.532cm}}{\pgfqpoint{0.137cm}{1.532cm}}
\pgfpathcurveto{\pgfqpoint{0.1cm}{1.532cm}}{\pgfqpoint{0.066cm}{1.518cm}}{\pgfqpoint{0.04cm}{1.492cm}}
\pgfpathcurveto{\pgfqpoint{0.014cm}{1.467cm}}{\pgfqpoint{0cm}{1.432cm}}{\pgfqpoint{0cm}{1.395cm}}
\pgfpathcurveto{\pgfqpoint{0cm}{1.359cm}}{\pgfqpoint{0.014cm}{1.324cm}}{\pgfqpoint{0.04cm}{1.299cm}}
\pgfpathcurveto{\pgfqpoint{0.066cm}{1.273cm}}{\pgfqpoint{0.1cm}{1.258cm}}{\pgfqpoint{0.137cm}{1.258cm}}
\pgfpathcurveto{\pgfqpoint{0.173cm}{1.258cm}}{\pgfqpoint{0.207cm}{1.273cm}}{\pgfqpoint{0.233cm}{1.299cm}}
\pgfpathcurveto{\pgfqpoint{0.259cm}{1.324cm}}{\pgfqpoint{0.273cm}{1.359cm}}{\pgfqpoint{0.273cm}{1.395cm}}
\pgfusepath{fill}
\begin{pgfscope}
\pgfsetdash{}{0cm}
\pgfsetlinewidth{0.818mm}
\pgfsetmiterlimit{7.0}
\pgfpathmoveto{\pgfqpoint{0.682cm}{0.671cm}}
\pgfpathlineto{\pgfqpoint{0.679cm}{1.418cm}}
\pgfusepath{stroke}
\end{pgfscope}
\pgfpathmoveto{\pgfqpoint{0.815cm}{1.399cm}}
\pgfpathcurveto{\pgfqpoint{0.815cm}{1.435cm}}{\pgfqpoint{0.801cm}{1.47cm}}{\pgfqpoint{0.775cm}{1.496cm}}
\pgfpathcurveto{\pgfqpoint{0.75cm}{1.521cm}}{\pgfqpoint{0.715cm}{1.536cm}}{\pgfqpoint{0.679cm}{1.536cm}}
\pgfpathcurveto{\pgfqpoint{0.643cm}{1.536cm}}{\pgfqpoint{0.608cm}{1.521cm}}{\pgfqpoint{0.582cm}{1.496cm}}
\pgfpathcurveto{\pgfqpoint{0.557cm}{1.47cm}}{\pgfqpoint{0.542cm}{1.435cm}}{\pgfqpoint{0.542cm}{1.399cm}}
\pgfpathcurveto{\pgfqpoint{0.542cm}{1.363cm}}{\pgfqpoint{0.557cm}{1.328cm}}{\pgfqpoint{0.582cm}{1.302cm}}
\pgfpathcurveto{\pgfqpoint{0.608cm}{1.276cm}}{\pgfqpoint{0.643cm}{1.262cm}}{\pgfqpoint{0.679cm}{1.262cm}}
\pgfpathcurveto{\pgfqpoint{0.715cm}{1.262cm}}{\pgfqpoint{0.75cm}{1.276cm}}{\pgfqpoint{0.775cm}{1.302cm}}
\pgfpathcurveto{\pgfqpoint{0.801cm}{1.328cm}}{\pgfqpoint{0.815cm}{1.363cm}}{\pgfqpoint{0.815cm}{1.399cm}}
\pgfusepath{fill}
\pgfpathmoveto{\pgfqpoint{1.345cm}{1.371cm}}
\pgfpathcurveto{\pgfqpoint{1.345cm}{1.408cm}}{\pgfqpoint{1.331cm}{1.442cm}}{\pgfqpoint{1.305cm}{1.468cm}}
\pgfpathcurveto{\pgfqpoint{1.28cm}{1.494cm}}{\pgfqpoint{1.245cm}{1.508cm}}{\pgfqpoint{1.209cm}{1.508cm}}
\pgfpathcurveto{\pgfqpoint{1.172cm}{1.508cm}}{\pgfqpoint{1.138cm}{1.494cm}}{\pgfqpoint{1.112cm}{1.468cm}}
\pgfpathcurveto{\pgfqpoint{1.087cm}{1.442cm}}{\pgfqpoint{1.072cm}{1.408cm}}{\pgfqpoint{1.072cm}{1.371cm}}
\pgfpathcurveto{\pgfqpoint{1.072cm}{1.335cm}}{\pgfqpoint{1.087cm}{1.3cm}}{\pgfqpoint{1.112cm}{1.274cm}}
\pgfpathcurveto{\pgfqpoint{1.138cm}{1.249cm}}{\pgfqpoint{1.172cm}{1.234cm}}{\pgfqpoint{1.209cm}{1.234cm}}
\pgfpathcurveto{\pgfqpoint{1.245cm}{1.234cm}}{\pgfqpoint{1.28cm}{1.249cm}}{\pgfqpoint{1.305cm}{1.274cm}}
\pgfpathcurveto{\pgfqpoint{1.331cm}{1.3cm}}{\pgfqpoint{1.345cm}{1.335cm}}{\pgfqpoint{1.345cm}{1.371cm}}
\pgfusepath{fill}
\begin{pgfscope}
\pgfsetdash{}{0cm}
\pgfsetlinewidth{0.818mm}
\pgfsetroundcap
\pgfsetmiterlimit{4.0}
\pgfpathmoveto{\pgfqpoint{0.682cm}{0.671cm}}
\pgfpathlineto{\pgfqpoint{0.682cm}{0.042cm}}
\pgfusepath{stroke}
\end{pgfscope}
\end{pgfscope}
\end{pgfscope}
\end{pgfscope}
\end{tikzpicture}}})}-\rmb{6X\circ(X^{\!\resizebox{0.6em}{!}{
\begin{tikzpicture}
\pgfpathmoveto{\pgfqpoint{0cm}{-0.035cm}}
\pgfpathlineto{\pgfqpoint{1.376cm}{-0.035cm}}
\pgfpathlineto{\pgfqpoint{1.376cm}{1.552cm}}
\pgfpathlineto{\pgfqpoint{0cm}{1.552cm}}
\pgfpathclose
\pgfusepath{clip}
\begin{pgfscope}
\begin{pgfscope}
\pgfpathmoveto{\pgfqpoint{0cm}{-0.035cm}}
\pgfpathlineto{\pgfqpoint{1.376cm}{-0.035cm}}
\pgfpathlineto{\pgfqpoint{1.376cm}{1.552cm}}
\pgfpathlineto{\pgfqpoint{0cm}{1.552cm}}
\pgfpathclose
\pgfusepath{clip}
\begin{pgfscope}
\begin{pgfscope}
\pgfsetdash{}{0cm}
\pgfsetlinewidth{0.818mm}
\pgfsetroundcap
\pgfsetroundjoin
\pgfsetmiterlimit{7.0}
\definecolor{eps2pgf_color}{gray}{0}\pgfsetstrokecolor{eps2pgf_color}\pgfsetfillcolor{eps2pgf_color}
\pgfpathmoveto{\pgfqpoint{0.117cm}{1.421cm}}
\pgfpathlineto{\pgfqpoint{0.682cm}{0.671cm}}
\pgfpathlineto{\pgfqpoint{1.246cm}{1.421cm}}
\pgfusepath{stroke}
\end{pgfscope}
\definecolor{eps2pgf_color}{gray}{0}\pgfsetstrokecolor{eps2pgf_color}\pgfsetfillcolor{eps2pgf_color}
\pgfpathmoveto{\pgfqpoint{0.273cm}{1.395cm}}
\pgfpathcurveto{\pgfqpoint{0.273cm}{1.432cm}}{\pgfqpoint{0.259cm}{1.467cm}}{\pgfqpoint{0.233cm}{1.492cm}}
\pgfpathcurveto{\pgfqpoint{0.207cm}{1.518cm}}{\pgfqpoint{0.173cm}{1.532cm}}{\pgfqpoint{0.137cm}{1.532cm}}
\pgfpathcurveto{\pgfqpoint{0.1cm}{1.532cm}}{\pgfqpoint{0.066cm}{1.518cm}}{\pgfqpoint{0.04cm}{1.492cm}}
\pgfpathcurveto{\pgfqpoint{0.014cm}{1.467cm}}{\pgfqpoint{0cm}{1.432cm}}{\pgfqpoint{0cm}{1.395cm}}
\pgfpathcurveto{\pgfqpoint{0cm}{1.359cm}}{\pgfqpoint{0.014cm}{1.324cm}}{\pgfqpoint{0.04cm}{1.299cm}}
\pgfpathcurveto{\pgfqpoint{0.066cm}{1.273cm}}{\pgfqpoint{0.1cm}{1.258cm}}{\pgfqpoint{0.137cm}{1.258cm}}
\pgfpathcurveto{\pgfqpoint{0.173cm}{1.258cm}}{\pgfqpoint{0.207cm}{1.273cm}}{\pgfqpoint{0.233cm}{1.299cm}}
\pgfpathcurveto{\pgfqpoint{0.259cm}{1.324cm}}{\pgfqpoint{0.273cm}{1.359cm}}{\pgfqpoint{0.273cm}{1.395cm}}
\pgfusepath{fill}
\begin{pgfscope}
\pgfsetdash{}{0cm}
\pgfsetlinewidth{0.818mm}
\pgfsetmiterlimit{7.0}
\pgfpathmoveto{\pgfqpoint{0.682cm}{0.671cm}}
\pgfpathlineto{\pgfqpoint{0.679cm}{1.418cm}}
\pgfusepath{stroke}
\end{pgfscope}
\pgfpathmoveto{\pgfqpoint{0.815cm}{1.399cm}}
\pgfpathcurveto{\pgfqpoint{0.815cm}{1.435cm}}{\pgfqpoint{0.801cm}{1.47cm}}{\pgfqpoint{0.775cm}{1.496cm}}
\pgfpathcurveto{\pgfqpoint{0.75cm}{1.521cm}}{\pgfqpoint{0.715cm}{1.536cm}}{\pgfqpoint{0.679cm}{1.536cm}}
\pgfpathcurveto{\pgfqpoint{0.643cm}{1.536cm}}{\pgfqpoint{0.608cm}{1.521cm}}{\pgfqpoint{0.582cm}{1.496cm}}
\pgfpathcurveto{\pgfqpoint{0.557cm}{1.47cm}}{\pgfqpoint{0.542cm}{1.435cm}}{\pgfqpoint{0.542cm}{1.399cm}}
\pgfpathcurveto{\pgfqpoint{0.542cm}{1.363cm}}{\pgfqpoint{0.557cm}{1.328cm}}{\pgfqpoint{0.582cm}{1.302cm}}
\pgfpathcurveto{\pgfqpoint{0.608cm}{1.276cm}}{\pgfqpoint{0.643cm}{1.262cm}}{\pgfqpoint{0.679cm}{1.262cm}}
\pgfpathcurveto{\pgfqpoint{0.715cm}{1.262cm}}{\pgfqpoint{0.75cm}{1.276cm}}{\pgfqpoint{0.775cm}{1.302cm}}
\pgfpathcurveto{\pgfqpoint{0.801cm}{1.328cm}}{\pgfqpoint{0.815cm}{1.363cm}}{\pgfqpoint{0.815cm}{1.399cm}}
\pgfusepath{fill}
\pgfpathmoveto{\pgfqpoint{1.345cm}{1.371cm}}
\pgfpathcurveto{\pgfqpoint{1.345cm}{1.408cm}}{\pgfqpoint{1.331cm}{1.442cm}}{\pgfqpoint{1.305cm}{1.468cm}}
\pgfpathcurveto{\pgfqpoint{1.28cm}{1.494cm}}{\pgfqpoint{1.245cm}{1.508cm}}{\pgfqpoint{1.209cm}{1.508cm}}
\pgfpathcurveto{\pgfqpoint{1.172cm}{1.508cm}}{\pgfqpoint{1.138cm}{1.494cm}}{\pgfqpoint{1.112cm}{1.468cm}}
\pgfpathcurveto{\pgfqpoint{1.087cm}{1.442cm}}{\pgfqpoint{1.072cm}{1.408cm}}{\pgfqpoint{1.072cm}{1.371cm}}
\pgfpathcurveto{\pgfqpoint{1.072cm}{1.335cm}}{\pgfqpoint{1.087cm}{1.3cm}}{\pgfqpoint{1.112cm}{1.274cm}}
\pgfpathcurveto{\pgfqpoint{1.138cm}{1.249cm}}{\pgfqpoint{1.172cm}{1.234cm}}{\pgfqpoint{1.209cm}{1.234cm}}
\pgfpathcurveto{\pgfqpoint{1.245cm}{1.234cm}}{\pgfqpoint{1.28cm}{1.249cm}}{\pgfqpoint{1.305cm}{1.274cm}}
\pgfpathcurveto{\pgfqpoint{1.331cm}{1.3cm}}{\pgfqpoint{1.345cm}{1.335cm}}{\pgfqpoint{1.345cm}{1.371cm}}
\pgfusepath{fill}
\begin{pgfscope}
\pgfsetdash{}{0cm}
\pgfsetlinewidth{0.818mm}
\pgfsetroundcap
\pgfsetmiterlimit{4.0}
\pgfpathmoveto{\pgfqpoint{0.682cm}{0.671cm}}
\pgfpathlineto{\pgfqpoint{0.682cm}{0.042cm}}
\pgfusepath{stroke}
\end{pgfscope}
\end{pgfscope}
\end{pgfscope}
\end{pgfscope}
\end{tikzpicture}}}\preccurlyeq(\phi+\psi))}-\rmb{6\mathrm{com}(\phi+\psi,X^{\!\resizebox{0.6em}{!}{
\begin{tikzpicture}
\pgfpathmoveto{\pgfqpoint{0cm}{-0.035cm}}
\pgfpathlineto{\pgfqpoint{1.376cm}{-0.035cm}}
\pgfpathlineto{\pgfqpoint{1.376cm}{1.552cm}}
\pgfpathlineto{\pgfqpoint{0cm}{1.552cm}}
\pgfpathclose
\pgfusepath{clip}
\begin{pgfscope}
\begin{pgfscope}
\pgfpathmoveto{\pgfqpoint{0cm}{-0.035cm}}
\pgfpathlineto{\pgfqpoint{1.376cm}{-0.035cm}}
\pgfpathlineto{\pgfqpoint{1.376cm}{1.552cm}}
\pgfpathlineto{\pgfqpoint{0cm}{1.552cm}}
\pgfpathclose
\pgfusepath{clip}
\begin{pgfscope}
\begin{pgfscope}
\pgfsetdash{}{0cm}
\pgfsetlinewidth{0.818mm}
\pgfsetroundcap
\pgfsetroundjoin
\pgfsetmiterlimit{7.0}
\definecolor{eps2pgf_color}{gray}{0}\pgfsetstrokecolor{eps2pgf_color}\pgfsetfillcolor{eps2pgf_color}
\pgfpathmoveto{\pgfqpoint{0.117cm}{1.421cm}}
\pgfpathlineto{\pgfqpoint{0.682cm}{0.671cm}}
\pgfpathlineto{\pgfqpoint{1.246cm}{1.421cm}}
\pgfusepath{stroke}
\end{pgfscope}
\definecolor{eps2pgf_color}{gray}{0}\pgfsetstrokecolor{eps2pgf_color}\pgfsetfillcolor{eps2pgf_color}
\pgfpathmoveto{\pgfqpoint{0.273cm}{1.395cm}}
\pgfpathcurveto{\pgfqpoint{0.273cm}{1.432cm}}{\pgfqpoint{0.259cm}{1.467cm}}{\pgfqpoint{0.233cm}{1.492cm}}
\pgfpathcurveto{\pgfqpoint{0.207cm}{1.518cm}}{\pgfqpoint{0.173cm}{1.532cm}}{\pgfqpoint{0.137cm}{1.532cm}}
\pgfpathcurveto{\pgfqpoint{0.1cm}{1.532cm}}{\pgfqpoint{0.066cm}{1.518cm}}{\pgfqpoint{0.04cm}{1.492cm}}
\pgfpathcurveto{\pgfqpoint{0.014cm}{1.467cm}}{\pgfqpoint{0cm}{1.432cm}}{\pgfqpoint{0cm}{1.395cm}}
\pgfpathcurveto{\pgfqpoint{0cm}{1.359cm}}{\pgfqpoint{0.014cm}{1.324cm}}{\pgfqpoint{0.04cm}{1.299cm}}
\pgfpathcurveto{\pgfqpoint{0.066cm}{1.273cm}}{\pgfqpoint{0.1cm}{1.258cm}}{\pgfqpoint{0.137cm}{1.258cm}}
\pgfpathcurveto{\pgfqpoint{0.173cm}{1.258cm}}{\pgfqpoint{0.207cm}{1.273cm}}{\pgfqpoint{0.233cm}{1.299cm}}
\pgfpathcurveto{\pgfqpoint{0.259cm}{1.324cm}}{\pgfqpoint{0.273cm}{1.359cm}}{\pgfqpoint{0.273cm}{1.395cm}}
\pgfusepath{fill}
\begin{pgfscope}
\pgfsetdash{}{0cm}
\pgfsetlinewidth{0.818mm}
\pgfsetmiterlimit{7.0}
\pgfpathmoveto{\pgfqpoint{0.682cm}{0.671cm}}
\pgfpathlineto{\pgfqpoint{0.679cm}{1.418cm}}
\pgfusepath{stroke}
\end{pgfscope}
\pgfpathmoveto{\pgfqpoint{0.815cm}{1.399cm}}
\pgfpathcurveto{\pgfqpoint{0.815cm}{1.435cm}}{\pgfqpoint{0.801cm}{1.47cm}}{\pgfqpoint{0.775cm}{1.496cm}}
\pgfpathcurveto{\pgfqpoint{0.75cm}{1.521cm}}{\pgfqpoint{0.715cm}{1.536cm}}{\pgfqpoint{0.679cm}{1.536cm}}
\pgfpathcurveto{\pgfqpoint{0.643cm}{1.536cm}}{\pgfqpoint{0.608cm}{1.521cm}}{\pgfqpoint{0.582cm}{1.496cm}}
\pgfpathcurveto{\pgfqpoint{0.557cm}{1.47cm}}{\pgfqpoint{0.542cm}{1.435cm}}{\pgfqpoint{0.542cm}{1.399cm}}
\pgfpathcurveto{\pgfqpoint{0.542cm}{1.363cm}}{\pgfqpoint{0.557cm}{1.328cm}}{\pgfqpoint{0.582cm}{1.302cm}}
\pgfpathcurveto{\pgfqpoint{0.608cm}{1.276cm}}{\pgfqpoint{0.643cm}{1.262cm}}{\pgfqpoint{0.679cm}{1.262cm}}
\pgfpathcurveto{\pgfqpoint{0.715cm}{1.262cm}}{\pgfqpoint{0.75cm}{1.276cm}}{\pgfqpoint{0.775cm}{1.302cm}}
\pgfpathcurveto{\pgfqpoint{0.801cm}{1.328cm}}{\pgfqpoint{0.815cm}{1.363cm}}{\pgfqpoint{0.815cm}{1.399cm}}
\pgfusepath{fill}
\pgfpathmoveto{\pgfqpoint{1.345cm}{1.371cm}}
\pgfpathcurveto{\pgfqpoint{1.345cm}{1.408cm}}{\pgfqpoint{1.331cm}{1.442cm}}{\pgfqpoint{1.305cm}{1.468cm}}
\pgfpathcurveto{\pgfqpoint{1.28cm}{1.494cm}}{\pgfqpoint{1.245cm}{1.508cm}}{\pgfqpoint{1.209cm}{1.508cm}}
\pgfpathcurveto{\pgfqpoint{1.172cm}{1.508cm}}{\pgfqpoint{1.138cm}{1.494cm}}{\pgfqpoint{1.112cm}{1.468cm}}
\pgfpathcurveto{\pgfqpoint{1.087cm}{1.442cm}}{\pgfqpoint{1.072cm}{1.408cm}}{\pgfqpoint{1.072cm}{1.371cm}}
\pgfpathcurveto{\pgfqpoint{1.072cm}{1.335cm}}{\pgfqpoint{1.087cm}{1.3cm}}{\pgfqpoint{1.112cm}{1.274cm}}
\pgfpathcurveto{\pgfqpoint{1.138cm}{1.249cm}}{\pgfqpoint{1.172cm}{1.234cm}}{\pgfqpoint{1.209cm}{1.234cm}}
\pgfpathcurveto{\pgfqpoint{1.245cm}{1.234cm}}{\pgfqpoint{1.28cm}{1.249cm}}{\pgfqpoint{1.305cm}{1.274cm}}
\pgfpathcurveto{\pgfqpoint{1.331cm}{1.3cm}}{\pgfqpoint{1.345cm}{1.335cm}}{\pgfqpoint{1.345cm}{1.371cm}}
\pgfusepath{fill}
\begin{pgfscope}
\pgfsetdash{}{0cm}
\pgfsetlinewidth{0.818mm}
\pgfsetroundcap
\pgfsetmiterlimit{4.0}
\pgfpathmoveto{\pgfqpoint{0.682cm}{0.671cm}}
\pgfpathlineto{\pgfqpoint{0.682cm}{0.042cm}}
\pgfusepath{stroke}
\end{pgfscope}
\end{pgfscope}
\end{pgfscope}
\end{pgfscope}
\end{tikzpicture}}},X)}\\
&\qquad+\rmb{(-X^{\!\resizebox{0.6em}{!}{
\begin{tikzpicture}
\pgfpathmoveto{\pgfqpoint{0cm}{-0.035cm}}
\pgfpathlineto{\pgfqpoint{1.376cm}{-0.035cm}}
\pgfpathlineto{\pgfqpoint{1.376cm}{1.552cm}}
\pgfpathlineto{\pgfqpoint{0cm}{1.552cm}}
\pgfpathclose
\pgfusepath{clip}
\begin{pgfscope}
\begin{pgfscope}
\pgfpathmoveto{\pgfqpoint{0cm}{-0.035cm}}
\pgfpathlineto{\pgfqpoint{1.376cm}{-0.035cm}}
\pgfpathlineto{\pgfqpoint{1.376cm}{1.552cm}}
\pgfpathlineto{\pgfqpoint{0cm}{1.552cm}}
\pgfpathclose
\pgfusepath{clip}
\begin{pgfscope}
\begin{pgfscope}
\pgfsetdash{}{0cm}
\pgfsetlinewidth{0.818mm}
\pgfsetroundcap
\pgfsetroundjoin
\pgfsetmiterlimit{7.0}
\definecolor{eps2pgf_color}{gray}{0}\pgfsetstrokecolor{eps2pgf_color}\pgfsetfillcolor{eps2pgf_color}
\pgfpathmoveto{\pgfqpoint{0.117cm}{1.421cm}}
\pgfpathlineto{\pgfqpoint{0.682cm}{0.671cm}}
\pgfpathlineto{\pgfqpoint{1.246cm}{1.421cm}}
\pgfusepath{stroke}
\end{pgfscope}
\definecolor{eps2pgf_color}{gray}{0}\pgfsetstrokecolor{eps2pgf_color}\pgfsetfillcolor{eps2pgf_color}
\pgfpathmoveto{\pgfqpoint{0.273cm}{1.395cm}}
\pgfpathcurveto{\pgfqpoint{0.273cm}{1.432cm}}{\pgfqpoint{0.259cm}{1.467cm}}{\pgfqpoint{0.233cm}{1.492cm}}
\pgfpathcurveto{\pgfqpoint{0.207cm}{1.518cm}}{\pgfqpoint{0.173cm}{1.532cm}}{\pgfqpoint{0.137cm}{1.532cm}}
\pgfpathcurveto{\pgfqpoint{0.1cm}{1.532cm}}{\pgfqpoint{0.066cm}{1.518cm}}{\pgfqpoint{0.04cm}{1.492cm}}
\pgfpathcurveto{\pgfqpoint{0.014cm}{1.467cm}}{\pgfqpoint{0cm}{1.432cm}}{\pgfqpoint{0cm}{1.395cm}}
\pgfpathcurveto{\pgfqpoint{0cm}{1.359cm}}{\pgfqpoint{0.014cm}{1.324cm}}{\pgfqpoint{0.04cm}{1.299cm}}
\pgfpathcurveto{\pgfqpoint{0.066cm}{1.273cm}}{\pgfqpoint{0.1cm}{1.258cm}}{\pgfqpoint{0.137cm}{1.258cm}}
\pgfpathcurveto{\pgfqpoint{0.173cm}{1.258cm}}{\pgfqpoint{0.207cm}{1.273cm}}{\pgfqpoint{0.233cm}{1.299cm}}
\pgfpathcurveto{\pgfqpoint{0.259cm}{1.324cm}}{\pgfqpoint{0.273cm}{1.359cm}}{\pgfqpoint{0.273cm}{1.395cm}}
\pgfusepath{fill}
\begin{pgfscope}
\pgfsetdash{}{0cm}
\pgfsetlinewidth{0.818mm}
\pgfsetmiterlimit{7.0}
\pgfpathmoveto{\pgfqpoint{0.682cm}{0.671cm}}
\pgfpathlineto{\pgfqpoint{0.679cm}{1.418cm}}
\pgfusepath{stroke}
\end{pgfscope}
\pgfpathmoveto{\pgfqpoint{0.815cm}{1.399cm}}
\pgfpathcurveto{\pgfqpoint{0.815cm}{1.435cm}}{\pgfqpoint{0.801cm}{1.47cm}}{\pgfqpoint{0.775cm}{1.496cm}}
\pgfpathcurveto{\pgfqpoint{0.75cm}{1.521cm}}{\pgfqpoint{0.715cm}{1.536cm}}{\pgfqpoint{0.679cm}{1.536cm}}
\pgfpathcurveto{\pgfqpoint{0.643cm}{1.536cm}}{\pgfqpoint{0.608cm}{1.521cm}}{\pgfqpoint{0.582cm}{1.496cm}}
\pgfpathcurveto{\pgfqpoint{0.557cm}{1.47cm}}{\pgfqpoint{0.542cm}{1.435cm}}{\pgfqpoint{0.542cm}{1.399cm}}
\pgfpathcurveto{\pgfqpoint{0.542cm}{1.363cm}}{\pgfqpoint{0.557cm}{1.328cm}}{\pgfqpoint{0.582cm}{1.302cm}}
\pgfpathcurveto{\pgfqpoint{0.608cm}{1.276cm}}{\pgfqpoint{0.643cm}{1.262cm}}{\pgfqpoint{0.679cm}{1.262cm}}
\pgfpathcurveto{\pgfqpoint{0.715cm}{1.262cm}}{\pgfqpoint{0.75cm}{1.276cm}}{\pgfqpoint{0.775cm}{1.302cm}}
\pgfpathcurveto{\pgfqpoint{0.801cm}{1.328cm}}{\pgfqpoint{0.815cm}{1.363cm}}{\pgfqpoint{0.815cm}{1.399cm}}
\pgfusepath{fill}
\pgfpathmoveto{\pgfqpoint{1.345cm}{1.371cm}}
\pgfpathcurveto{\pgfqpoint{1.345cm}{1.408cm}}{\pgfqpoint{1.331cm}{1.442cm}}{\pgfqpoint{1.305cm}{1.468cm}}
\pgfpathcurveto{\pgfqpoint{1.28cm}{1.494cm}}{\pgfqpoint{1.245cm}{1.508cm}}{\pgfqpoint{1.209cm}{1.508cm}}
\pgfpathcurveto{\pgfqpoint{1.172cm}{1.508cm}}{\pgfqpoint{1.138cm}{1.494cm}}{\pgfqpoint{1.112cm}{1.468cm}}
\pgfpathcurveto{\pgfqpoint{1.087cm}{1.442cm}}{\pgfqpoint{1.072cm}{1.408cm}}{\pgfqpoint{1.072cm}{1.371cm}}
\pgfpathcurveto{\pgfqpoint{1.072cm}{1.335cm}}{\pgfqpoint{1.087cm}{1.3cm}}{\pgfqpoint{1.112cm}{1.274cm}}
\pgfpathcurveto{\pgfqpoint{1.138cm}{1.249cm}}{\pgfqpoint{1.172cm}{1.234cm}}{\pgfqpoint{1.209cm}{1.234cm}}
\pgfpathcurveto{\pgfqpoint{1.245cm}{1.234cm}}{\pgfqpoint{1.28cm}{1.249cm}}{\pgfqpoint{1.305cm}{1.274cm}}
\pgfpathcurveto{\pgfqpoint{1.331cm}{1.3cm}}{\pgfqpoint{1.345cm}{1.335cm}}{\pgfqpoint{1.345cm}{1.371cm}}
\pgfusepath{fill}
\begin{pgfscope}
\pgfsetdash{}{0cm}
\pgfsetlinewidth{0.818mm}
\pgfsetroundcap
\pgfsetmiterlimit{4.0}
\pgfpathmoveto{\pgfqpoint{0.682cm}{0.671cm}}
\pgfpathlineto{\pgfqpoint{0.682cm}{0.042cm}}
\pgfusepath{stroke}
\end{pgfscope}
\end{pgfscope}
\end{pgfscope}
\end{pgfscope}
\end{tikzpicture}}}+\phi)^3}+\rmb{3(-X^{\!\resizebox{0.6em}{!}{
\begin{tikzpicture}
\pgfpathmoveto{\pgfqpoint{0cm}{-0.035cm}}
\pgfpathlineto{\pgfqpoint{1.376cm}{-0.035cm}}
\pgfpathlineto{\pgfqpoint{1.376cm}{1.552cm}}
\pgfpathlineto{\pgfqpoint{0cm}{1.552cm}}
\pgfpathclose
\pgfusepath{clip}
\begin{pgfscope}
\begin{pgfscope}
\pgfpathmoveto{\pgfqpoint{0cm}{-0.035cm}}
\pgfpathlineto{\pgfqpoint{1.376cm}{-0.035cm}}
\pgfpathlineto{\pgfqpoint{1.376cm}{1.552cm}}
\pgfpathlineto{\pgfqpoint{0cm}{1.552cm}}
\pgfpathclose
\pgfusepath{clip}
\begin{pgfscope}
\begin{pgfscope}
\pgfsetdash{}{0cm}
\pgfsetlinewidth{0.818mm}
\pgfsetroundcap
\pgfsetroundjoin
\pgfsetmiterlimit{7.0}
\definecolor{eps2pgf_color}{gray}{0}\pgfsetstrokecolor{eps2pgf_color}\pgfsetfillcolor{eps2pgf_color}
\pgfpathmoveto{\pgfqpoint{0.117cm}{1.421cm}}
\pgfpathlineto{\pgfqpoint{0.682cm}{0.671cm}}
\pgfpathlineto{\pgfqpoint{1.246cm}{1.421cm}}
\pgfusepath{stroke}
\end{pgfscope}
\definecolor{eps2pgf_color}{gray}{0}\pgfsetstrokecolor{eps2pgf_color}\pgfsetfillcolor{eps2pgf_color}
\pgfpathmoveto{\pgfqpoint{0.273cm}{1.395cm}}
\pgfpathcurveto{\pgfqpoint{0.273cm}{1.432cm}}{\pgfqpoint{0.259cm}{1.467cm}}{\pgfqpoint{0.233cm}{1.492cm}}
\pgfpathcurveto{\pgfqpoint{0.207cm}{1.518cm}}{\pgfqpoint{0.173cm}{1.532cm}}{\pgfqpoint{0.137cm}{1.532cm}}
\pgfpathcurveto{\pgfqpoint{0.1cm}{1.532cm}}{\pgfqpoint{0.066cm}{1.518cm}}{\pgfqpoint{0.04cm}{1.492cm}}
\pgfpathcurveto{\pgfqpoint{0.014cm}{1.467cm}}{\pgfqpoint{0cm}{1.432cm}}{\pgfqpoint{0cm}{1.395cm}}
\pgfpathcurveto{\pgfqpoint{0cm}{1.359cm}}{\pgfqpoint{0.014cm}{1.324cm}}{\pgfqpoint{0.04cm}{1.299cm}}
\pgfpathcurveto{\pgfqpoint{0.066cm}{1.273cm}}{\pgfqpoint{0.1cm}{1.258cm}}{\pgfqpoint{0.137cm}{1.258cm}}
\pgfpathcurveto{\pgfqpoint{0.173cm}{1.258cm}}{\pgfqpoint{0.207cm}{1.273cm}}{\pgfqpoint{0.233cm}{1.299cm}}
\pgfpathcurveto{\pgfqpoint{0.259cm}{1.324cm}}{\pgfqpoint{0.273cm}{1.359cm}}{\pgfqpoint{0.273cm}{1.395cm}}
\pgfusepath{fill}
\begin{pgfscope}
\pgfsetdash{}{0cm}
\pgfsetlinewidth{0.818mm}
\pgfsetmiterlimit{7.0}
\pgfpathmoveto{\pgfqpoint{0.682cm}{0.671cm}}
\pgfpathlineto{\pgfqpoint{0.679cm}{1.418cm}}
\pgfusepath{stroke}
\end{pgfscope}
\pgfpathmoveto{\pgfqpoint{0.815cm}{1.399cm}}
\pgfpathcurveto{\pgfqpoint{0.815cm}{1.435cm}}{\pgfqpoint{0.801cm}{1.47cm}}{\pgfqpoint{0.775cm}{1.496cm}}
\pgfpathcurveto{\pgfqpoint{0.75cm}{1.521cm}}{\pgfqpoint{0.715cm}{1.536cm}}{\pgfqpoint{0.679cm}{1.536cm}}
\pgfpathcurveto{\pgfqpoint{0.643cm}{1.536cm}}{\pgfqpoint{0.608cm}{1.521cm}}{\pgfqpoint{0.582cm}{1.496cm}}
\pgfpathcurveto{\pgfqpoint{0.557cm}{1.47cm}}{\pgfqpoint{0.542cm}{1.435cm}}{\pgfqpoint{0.542cm}{1.399cm}}
\pgfpathcurveto{\pgfqpoint{0.542cm}{1.363cm}}{\pgfqpoint{0.557cm}{1.328cm}}{\pgfqpoint{0.582cm}{1.302cm}}
\pgfpathcurveto{\pgfqpoint{0.608cm}{1.276cm}}{\pgfqpoint{0.643cm}{1.262cm}}{\pgfqpoint{0.679cm}{1.262cm}}
\pgfpathcurveto{\pgfqpoint{0.715cm}{1.262cm}}{\pgfqpoint{0.75cm}{1.276cm}}{\pgfqpoint{0.775cm}{1.302cm}}
\pgfpathcurveto{\pgfqpoint{0.801cm}{1.328cm}}{\pgfqpoint{0.815cm}{1.363cm}}{\pgfqpoint{0.815cm}{1.399cm}}
\pgfusepath{fill}
\pgfpathmoveto{\pgfqpoint{1.345cm}{1.371cm}}
\pgfpathcurveto{\pgfqpoint{1.345cm}{1.408cm}}{\pgfqpoint{1.331cm}{1.442cm}}{\pgfqpoint{1.305cm}{1.468cm}}
\pgfpathcurveto{\pgfqpoint{1.28cm}{1.494cm}}{\pgfqpoint{1.245cm}{1.508cm}}{\pgfqpoint{1.209cm}{1.508cm}}
\pgfpathcurveto{\pgfqpoint{1.172cm}{1.508cm}}{\pgfqpoint{1.138cm}{1.494cm}}{\pgfqpoint{1.112cm}{1.468cm}}
\pgfpathcurveto{\pgfqpoint{1.087cm}{1.442cm}}{\pgfqpoint{1.072cm}{1.408cm}}{\pgfqpoint{1.072cm}{1.371cm}}
\pgfpathcurveto{\pgfqpoint{1.072cm}{1.335cm}}{\pgfqpoint{1.087cm}{1.3cm}}{\pgfqpoint{1.112cm}{1.274cm}}
\pgfpathcurveto{\pgfqpoint{1.138cm}{1.249cm}}{\pgfqpoint{1.172cm}{1.234cm}}{\pgfqpoint{1.209cm}{1.234cm}}
\pgfpathcurveto{\pgfqpoint{1.245cm}{1.234cm}}{\pgfqpoint{1.28cm}{1.249cm}}{\pgfqpoint{1.305cm}{1.274cm}}
\pgfpathcurveto{\pgfqpoint{1.331cm}{1.3cm}}{\pgfqpoint{1.345cm}{1.335cm}}{\pgfqpoint{1.345cm}{1.371cm}}
\pgfusepath{fill}
\begin{pgfscope}
\pgfsetdash{}{0cm}
\pgfsetlinewidth{0.818mm}
\pgfsetroundcap
\pgfsetmiterlimit{4.0}
\pgfpathmoveto{\pgfqpoint{0.682cm}{0.671cm}}
\pgfpathlineto{\pgfqpoint{0.682cm}{0.042cm}}
\pgfusepath{stroke}
\end{pgfscope}
\end{pgfscope}
\end{pgfscope}
\end{pgfscope}
\end{tikzpicture}}}+\phi)^2\psi}+\rmb{3(-X^{\!\resizebox{0.6em}{!}{
\begin{tikzpicture}
\pgfpathmoveto{\pgfqpoint{0cm}{-0.035cm}}
\pgfpathlineto{\pgfqpoint{1.376cm}{-0.035cm}}
\pgfpathlineto{\pgfqpoint{1.376cm}{1.552cm}}
\pgfpathlineto{\pgfqpoint{0cm}{1.552cm}}
\pgfpathclose
\pgfusepath{clip}
\begin{pgfscope}
\begin{pgfscope}
\pgfpathmoveto{\pgfqpoint{0cm}{-0.035cm}}
\pgfpathlineto{\pgfqpoint{1.376cm}{-0.035cm}}
\pgfpathlineto{\pgfqpoint{1.376cm}{1.552cm}}
\pgfpathlineto{\pgfqpoint{0cm}{1.552cm}}
\pgfpathclose
\pgfusepath{clip}
\begin{pgfscope}
\begin{pgfscope}
\pgfsetdash{}{0cm}
\pgfsetlinewidth{0.818mm}
\pgfsetroundcap
\pgfsetroundjoin
\pgfsetmiterlimit{7.0}
\definecolor{eps2pgf_color}{gray}{0}\pgfsetstrokecolor{eps2pgf_color}\pgfsetfillcolor{eps2pgf_color}
\pgfpathmoveto{\pgfqpoint{0.117cm}{1.421cm}}
\pgfpathlineto{\pgfqpoint{0.682cm}{0.671cm}}
\pgfpathlineto{\pgfqpoint{1.246cm}{1.421cm}}
\pgfusepath{stroke}
\end{pgfscope}
\definecolor{eps2pgf_color}{gray}{0}\pgfsetstrokecolor{eps2pgf_color}\pgfsetfillcolor{eps2pgf_color}
\pgfpathmoveto{\pgfqpoint{0.273cm}{1.395cm}}
\pgfpathcurveto{\pgfqpoint{0.273cm}{1.432cm}}{\pgfqpoint{0.259cm}{1.467cm}}{\pgfqpoint{0.233cm}{1.492cm}}
\pgfpathcurveto{\pgfqpoint{0.207cm}{1.518cm}}{\pgfqpoint{0.173cm}{1.532cm}}{\pgfqpoint{0.137cm}{1.532cm}}
\pgfpathcurveto{\pgfqpoint{0.1cm}{1.532cm}}{\pgfqpoint{0.066cm}{1.518cm}}{\pgfqpoint{0.04cm}{1.492cm}}
\pgfpathcurveto{\pgfqpoint{0.014cm}{1.467cm}}{\pgfqpoint{0cm}{1.432cm}}{\pgfqpoint{0cm}{1.395cm}}
\pgfpathcurveto{\pgfqpoint{0cm}{1.359cm}}{\pgfqpoint{0.014cm}{1.324cm}}{\pgfqpoint{0.04cm}{1.299cm}}
\pgfpathcurveto{\pgfqpoint{0.066cm}{1.273cm}}{\pgfqpoint{0.1cm}{1.258cm}}{\pgfqpoint{0.137cm}{1.258cm}}
\pgfpathcurveto{\pgfqpoint{0.173cm}{1.258cm}}{\pgfqpoint{0.207cm}{1.273cm}}{\pgfqpoint{0.233cm}{1.299cm}}
\pgfpathcurveto{\pgfqpoint{0.259cm}{1.324cm}}{\pgfqpoint{0.273cm}{1.359cm}}{\pgfqpoint{0.273cm}{1.395cm}}
\pgfusepath{fill}
\begin{pgfscope}
\pgfsetdash{}{0cm}
\pgfsetlinewidth{0.818mm}
\pgfsetmiterlimit{7.0}
\pgfpathmoveto{\pgfqpoint{0.682cm}{0.671cm}}
\pgfpathlineto{\pgfqpoint{0.679cm}{1.418cm}}
\pgfusepath{stroke}
\end{pgfscope}
\pgfpathmoveto{\pgfqpoint{0.815cm}{1.399cm}}
\pgfpathcurveto{\pgfqpoint{0.815cm}{1.435cm}}{\pgfqpoint{0.801cm}{1.47cm}}{\pgfqpoint{0.775cm}{1.496cm}}
\pgfpathcurveto{\pgfqpoint{0.75cm}{1.521cm}}{\pgfqpoint{0.715cm}{1.536cm}}{\pgfqpoint{0.679cm}{1.536cm}}
\pgfpathcurveto{\pgfqpoint{0.643cm}{1.536cm}}{\pgfqpoint{0.608cm}{1.521cm}}{\pgfqpoint{0.582cm}{1.496cm}}
\pgfpathcurveto{\pgfqpoint{0.557cm}{1.47cm}}{\pgfqpoint{0.542cm}{1.435cm}}{\pgfqpoint{0.542cm}{1.399cm}}
\pgfpathcurveto{\pgfqpoint{0.542cm}{1.363cm}}{\pgfqpoint{0.557cm}{1.328cm}}{\pgfqpoint{0.582cm}{1.302cm}}
\pgfpathcurveto{\pgfqpoint{0.608cm}{1.276cm}}{\pgfqpoint{0.643cm}{1.262cm}}{\pgfqpoint{0.679cm}{1.262cm}}
\pgfpathcurveto{\pgfqpoint{0.715cm}{1.262cm}}{\pgfqpoint{0.75cm}{1.276cm}}{\pgfqpoint{0.775cm}{1.302cm}}
\pgfpathcurveto{\pgfqpoint{0.801cm}{1.328cm}}{\pgfqpoint{0.815cm}{1.363cm}}{\pgfqpoint{0.815cm}{1.399cm}}
\pgfusepath{fill}
\pgfpathmoveto{\pgfqpoint{1.345cm}{1.371cm}}
\pgfpathcurveto{\pgfqpoint{1.345cm}{1.408cm}}{\pgfqpoint{1.331cm}{1.442cm}}{\pgfqpoint{1.305cm}{1.468cm}}
\pgfpathcurveto{\pgfqpoint{1.28cm}{1.494cm}}{\pgfqpoint{1.245cm}{1.508cm}}{\pgfqpoint{1.209cm}{1.508cm}}
\pgfpathcurveto{\pgfqpoint{1.172cm}{1.508cm}}{\pgfqpoint{1.138cm}{1.494cm}}{\pgfqpoint{1.112cm}{1.468cm}}
\pgfpathcurveto{\pgfqpoint{1.087cm}{1.442cm}}{\pgfqpoint{1.072cm}{1.408cm}}{\pgfqpoint{1.072cm}{1.371cm}}
\pgfpathcurveto{\pgfqpoint{1.072cm}{1.335cm}}{\pgfqpoint{1.087cm}{1.3cm}}{\pgfqpoint{1.112cm}{1.274cm}}
\pgfpathcurveto{\pgfqpoint{1.138cm}{1.249cm}}{\pgfqpoint{1.172cm}{1.234cm}}{\pgfqpoint{1.209cm}{1.234cm}}
\pgfpathcurveto{\pgfqpoint{1.245cm}{1.234cm}}{\pgfqpoint{1.28cm}{1.249cm}}{\pgfqpoint{1.305cm}{1.274cm}}
\pgfpathcurveto{\pgfqpoint{1.331cm}{1.3cm}}{\pgfqpoint{1.345cm}{1.335cm}}{\pgfqpoint{1.345cm}{1.371cm}}
\pgfusepath{fill}
\begin{pgfscope}
\pgfsetdash{}{0cm}
\pgfsetlinewidth{0.818mm}
\pgfsetroundcap
\pgfsetmiterlimit{4.0}
\pgfpathmoveto{\pgfqpoint{0.682cm}{0.671cm}}
\pgfpathlineto{\pgfqpoint{0.682cm}{0.042cm}}
\pgfusepath{stroke}
\end{pgfscope}
\end{pgfscope}
\end{pgfscope}
\end{pgfscope}
\end{tikzpicture}}}+\phi)\psi^2}\\
&\qquad +\rmb{3\UU_\leq\llbracket X^2 \rrbracket\succ(\phi+\psi)}
+\rmb{3\UU_\leqslant\llbracket X^2 \rrbracket\prec(\phi+\psi)}-\rmb{3( \phi + \psi)\prec\UU_\leqslantX^{\!\resizebox{!}{.8em}{
\begin{tikzpicture}
\pgfpathmoveto{\pgfqpoint{0cm}{-0.035cm}}
\pgfpathlineto{\pgfqpoint{1.976cm}{-0.035cm}}
\pgfpathlineto{\pgfqpoint{1.976cm}{1.94cm}}
\pgfpathlineto{\pgfqpoint{0cm}{1.94cm}}
\pgfpathclose
\pgfusepath{clip}
\begin{pgfscope}
\begin{pgfscope}
\pgfpathmoveto{\pgfqpoint{0cm}{-0.035cm}}
\pgfpathlineto{\pgfqpoint{1.976cm}{-0.035cm}}
\pgfpathlineto{\pgfqpoint{1.976cm}{1.94cm}}
\pgfpathlineto{\pgfqpoint{0cm}{1.94cm}}
\pgfpathclose
\pgfusepath{clip}
\begin{pgfscope}
\begin{pgfscope}
\pgfsetdash{}{0cm}
\pgfsetlinewidth{0.818mm}
\pgfsetroundcap
\pgfsetroundjoin
\pgfsetmiterlimit{7.0}
\definecolor{eps2pgf_color}{gray}{0}\pgfsetstrokecolor{eps2pgf_color}\pgfsetfillcolor{eps2pgf_color}
\pgfpathmoveto{\pgfqpoint{0.117cm}{1.815cm}}
\pgfpathlineto{\pgfqpoint{0.682cm}{1.065cm}}
\pgfpathlineto{\pgfqpoint{1.246cm}{1.815cm}}
\pgfusepath{stroke}
\end{pgfscope}
\definecolor{eps2pgf_color}{gray}{0}\pgfsetstrokecolor{eps2pgf_color}\pgfsetfillcolor{eps2pgf_color}
\pgfpathmoveto{\pgfqpoint{0.273cm}{1.789cm}}
\pgfpathcurveto{\pgfqpoint{0.273cm}{1.825cm}}{\pgfqpoint{0.259cm}{1.86cm}}{\pgfqpoint{0.233cm}{1.886cm}}
\pgfpathcurveto{\pgfqpoint{0.207cm}{1.912cm}}{\pgfqpoint{0.173cm}{1.926cm}}{\pgfqpoint{0.137cm}{1.926cm}}
\pgfpathcurveto{\pgfqpoint{0.1cm}{1.926cm}}{\pgfqpoint{0.066cm}{1.912cm}}{\pgfqpoint{0.04cm}{1.886cm}}
\pgfpathcurveto{\pgfqpoint{0.014cm}{1.86cm}}{\pgfqpoint{0cm}{1.825cm}}{\pgfqpoint{0cm}{1.789cm}}
\pgfpathcurveto{\pgfqpoint{0cm}{1.753cm}}{\pgfqpoint{0.014cm}{1.718cm}}{\pgfqpoint{0.04cm}{1.692cm}}
\pgfpathcurveto{\pgfqpoint{0.066cm}{1.667cm}}{\pgfqpoint{0.1cm}{1.652cm}}{\pgfqpoint{0.137cm}{1.652cm}}
\pgfpathcurveto{\pgfqpoint{0.173cm}{1.652cm}}{\pgfqpoint{0.207cm}{1.667cm}}{\pgfqpoint{0.233cm}{1.692cm}}
\pgfpathcurveto{\pgfqpoint{0.259cm}{1.718cm}}{\pgfqpoint{0.273cm}{1.753cm}}{\pgfqpoint{0.273cm}{1.789cm}}
\pgfusepath{fill}
\pgfpathmoveto{\pgfqpoint{1.345cm}{1.765cm}}
\pgfpathcurveto{\pgfqpoint{1.345cm}{1.801cm}}{\pgfqpoint{1.331cm}{1.836cm}}{\pgfqpoint{1.305cm}{1.862cm}}
\pgfpathcurveto{\pgfqpoint{1.28cm}{1.887cm}}{\pgfqpoint{1.245cm}{1.902cm}}{\pgfqpoint{1.209cm}{1.902cm}}
\pgfpathcurveto{\pgfqpoint{1.172cm}{1.902cm}}{\pgfqpoint{1.138cm}{1.887cm}}{\pgfqpoint{1.112cm}{1.862cm}}
\pgfpathcurveto{\pgfqpoint{1.087cm}{1.836cm}}{\pgfqpoint{1.072cm}{1.801cm}}{\pgfqpoint{1.072cm}{1.765cm}}
\pgfpathcurveto{\pgfqpoint{1.072cm}{1.728cm}}{\pgfqpoint{1.087cm}{1.694cm}}{\pgfqpoint{1.112cm}{1.668cm}}
\pgfpathcurveto{\pgfqpoint{1.138cm}{1.642cm}}{\pgfqpoint{1.172cm}{1.628cm}}{\pgfqpoint{1.209cm}{1.628cm}}
\pgfpathcurveto{\pgfqpoint{1.245cm}{1.628cm}}{\pgfqpoint{1.28cm}{1.642cm}}{\pgfqpoint{1.305cm}{1.668cm}}
\pgfpathcurveto{\pgfqpoint{1.331cm}{1.694cm}}{\pgfqpoint{1.345cm}{1.728cm}}{\pgfqpoint{1.345cm}{1.765cm}}
\pgfusepath{fill}
\begin{pgfscope}
\pgfsetdash{}{0cm}
\pgfsetlinewidth{0.818mm}
\pgfsetroundcap
\pgfsetroundjoin
\pgfsetmiterlimit{7.0}
\pgfpathmoveto{\pgfqpoint{0.682cm}{1.065cm}}
\pgfpathlineto{\pgfqpoint{1.246cm}{0.315cm}}
\pgfpathlineto{\pgfqpoint{1.811cm}{1.065cm}}
\pgfusepath{stroke}
\end{pgfscope}
\pgfpathmoveto{\pgfqpoint{1.948cm}{1.065cm}}
\pgfpathcurveto{\pgfqpoint{1.948cm}{1.101cm}}{\pgfqpoint{1.933cm}{1.136cm}}{\pgfqpoint{1.907cm}{1.162cm}}
\pgfpathcurveto{\pgfqpoint{1.882cm}{1.187cm}}{\pgfqpoint{1.847cm}{1.202cm}}{\pgfqpoint{1.811cm}{1.202cm}}
\pgfpathcurveto{\pgfqpoint{1.775cm}{1.202cm}}{\pgfqpoint{1.74cm}{1.187cm}}{\pgfqpoint{1.714cm}{1.162cm}}
\pgfpathcurveto{\pgfqpoint{1.689cm}{1.136cm}}{\pgfqpoint{1.674cm}{1.101cm}}{\pgfqpoint{1.674cm}{1.065cm}}
\pgfpathcurveto{\pgfqpoint{1.674cm}{1.029cm}}{\pgfqpoint{1.689cm}{0.994cm}}{\pgfqpoint{1.714cm}{0.968cm}}
\pgfpathcurveto{\pgfqpoint{1.74cm}{0.942cm}}{\pgfqpoint{1.775cm}{0.928cm}}{\pgfqpoint{1.811cm}{0.928cm}}
\pgfpathcurveto{\pgfqpoint{1.847cm}{0.928cm}}{\pgfqpoint{1.882cm}{0.942cm}}{\pgfqpoint{1.907cm}{0.968cm}}
\pgfpathcurveto{\pgfqpoint{1.933cm}{0.994cm}}{\pgfqpoint{1.948cm}{1.029cm}}{\pgfqpoint{1.948cm}{1.065cm}}
\pgfusepath{fill}
\begin{pgfscope}
\pgfsetdash{}{0cm}
\pgfsetlinewidth{0.818mm}
\pgfsetmiterlimit{7.0}
\pgfpathmoveto{\pgfqpoint{1.246cm}{0.315cm}}
\pgfpathlineto{\pgfqpoint{1.244cm}{1.061cm}}
\pgfusepath{stroke}
\end{pgfscope}
\pgfpathmoveto{\pgfqpoint{1.38cm}{1.065cm}}
\pgfpathcurveto{\pgfqpoint{1.38cm}{1.101cm}}{\pgfqpoint{1.366cm}{1.136cm}}{\pgfqpoint{1.34cm}{1.162cm}}
\pgfpathcurveto{\pgfqpoint{1.315cm}{1.187cm}}{\pgfqpoint{1.28cm}{1.202cm}}{\pgfqpoint{1.244cm}{1.202cm}}
\pgfpathcurveto{\pgfqpoint{1.207cm}{1.202cm}}{\pgfqpoint{1.173cm}{1.187cm}}{\pgfqpoint{1.147cm}{1.162cm}}
\pgfpathcurveto{\pgfqpoint{1.121cm}{1.136cm}}{\pgfqpoint{1.107cm}{1.101cm}}{\pgfqpoint{1.107cm}{1.065cm}}
\pgfpathcurveto{\pgfqpoint{1.107cm}{1.029cm}}{\pgfqpoint{1.121cm}{0.994cm}}{\pgfqpoint{1.147cm}{0.968cm}}
\pgfpathcurveto{\pgfqpoint{1.173cm}{0.942cm}}{\pgfqpoint{1.207cm}{0.928cm}}{\pgfqpoint{1.244cm}{0.928cm}}
\pgfpathcurveto{\pgfqpoint{1.28cm}{0.928cm}}{\pgfqpoint{1.315cm}{0.942cm}}{\pgfqpoint{1.34cm}{0.968cm}}
\pgfpathcurveto{\pgfqpoint{1.366cm}{0.994cm}}{\pgfqpoint{1.38cm}{1.029cm}}{\pgfqpoint{1.38cm}{1.065cm}}
\pgfusepath{fill}
\begin{pgfscope}
\pgfsetdash{}{0cm}
\pgfsetlinewidth{0.818mm}
\pgfsetmiterlimit{4.0}
\pgfpathmoveto{\pgfqpoint{1.383cm}{0.178cm}}
\pgfpathcurveto{\pgfqpoint{1.383cm}{0.214cm}}{\pgfqpoint{1.369cm}{0.249cm}}{\pgfqpoint{1.343cm}{0.275cm}}
\pgfpathcurveto{\pgfqpoint{1.317cm}{0.3cm}}{\pgfqpoint{1.283cm}{0.315cm}}{\pgfqpoint{1.246cm}{0.315cm}}
\pgfpathcurveto{\pgfqpoint{1.21cm}{0.315cm}}{\pgfqpoint{1.175cm}{0.3cm}}{\pgfqpoint{1.15cm}{0.275cm}}
\pgfpathcurveto{\pgfqpoint{1.124cm}{0.249cm}}{\pgfqpoint{1.11cm}{0.214cm}}{\pgfqpoint{1.11cm}{0.178cm}}
\pgfpathcurveto{\pgfqpoint{1.11cm}{0.141cm}}{\pgfqpoint{1.124cm}{0.107cm}}{\pgfqpoint{1.15cm}{0.081cm}}
\pgfpathcurveto{\pgfqpoint{1.175cm}{0.055cm}}{\pgfqpoint{1.21cm}{0.041cm}}{\pgfqpoint{1.246cm}{0.041cm}}
\pgfpathcurveto{\pgfqpoint{1.283cm}{0.041cm}}{\pgfqpoint{1.317cm}{0.055cm}}{\pgfqpoint{1.343cm}{0.081cm}}
\pgfpathcurveto{\pgfqpoint{1.369cm}{0.107cm}}{\pgfqpoint{1.383cm}{0.141cm}}{\pgfqpoint{1.383cm}{0.178cm}}
\pgfusepath{stroke}
\end{pgfscope}
\end{pgfscope}
\end{pgfscope}
\end{pgfscope}
\end{tikzpicture}}}}-\rmb{3( \phi + \psi)\succcurlyeqX^{\!\resizebox{!}{.8em}{
\begin{tikzpicture}
\pgfpathmoveto{\pgfqpoint{0cm}{-0.035cm}}
\pgfpathlineto{\pgfqpoint{1.976cm}{-0.035cm}}
\pgfpathlineto{\pgfqpoint{1.976cm}{1.94cm}}
\pgfpathlineto{\pgfqpoint{0cm}{1.94cm}}
\pgfpathclose
\pgfusepath{clip}
\begin{pgfscope}
\begin{pgfscope}
\pgfpathmoveto{\pgfqpoint{0cm}{-0.035cm}}
\pgfpathlineto{\pgfqpoint{1.976cm}{-0.035cm}}
\pgfpathlineto{\pgfqpoint{1.976cm}{1.94cm}}
\pgfpathlineto{\pgfqpoint{0cm}{1.94cm}}
\pgfpathclose
\pgfusepath{clip}
\begin{pgfscope}
\begin{pgfscope}
\pgfsetdash{}{0cm}
\pgfsetlinewidth{0.818mm}
\pgfsetroundcap
\pgfsetroundjoin
\pgfsetmiterlimit{7.0}
\definecolor{eps2pgf_color}{gray}{0}\pgfsetstrokecolor{eps2pgf_color}\pgfsetfillcolor{eps2pgf_color}
\pgfpathmoveto{\pgfqpoint{0.117cm}{1.815cm}}
\pgfpathlineto{\pgfqpoint{0.682cm}{1.065cm}}
\pgfpathlineto{\pgfqpoint{1.246cm}{1.815cm}}
\pgfusepath{stroke}
\end{pgfscope}
\definecolor{eps2pgf_color}{gray}{0}\pgfsetstrokecolor{eps2pgf_color}\pgfsetfillcolor{eps2pgf_color}
\pgfpathmoveto{\pgfqpoint{0.273cm}{1.789cm}}
\pgfpathcurveto{\pgfqpoint{0.273cm}{1.825cm}}{\pgfqpoint{0.259cm}{1.86cm}}{\pgfqpoint{0.233cm}{1.886cm}}
\pgfpathcurveto{\pgfqpoint{0.207cm}{1.912cm}}{\pgfqpoint{0.173cm}{1.926cm}}{\pgfqpoint{0.137cm}{1.926cm}}
\pgfpathcurveto{\pgfqpoint{0.1cm}{1.926cm}}{\pgfqpoint{0.066cm}{1.912cm}}{\pgfqpoint{0.04cm}{1.886cm}}
\pgfpathcurveto{\pgfqpoint{0.014cm}{1.86cm}}{\pgfqpoint{0cm}{1.825cm}}{\pgfqpoint{0cm}{1.789cm}}
\pgfpathcurveto{\pgfqpoint{0cm}{1.753cm}}{\pgfqpoint{0.014cm}{1.718cm}}{\pgfqpoint{0.04cm}{1.692cm}}
\pgfpathcurveto{\pgfqpoint{0.066cm}{1.667cm}}{\pgfqpoint{0.1cm}{1.652cm}}{\pgfqpoint{0.137cm}{1.652cm}}
\pgfpathcurveto{\pgfqpoint{0.173cm}{1.652cm}}{\pgfqpoint{0.207cm}{1.667cm}}{\pgfqpoint{0.233cm}{1.692cm}}
\pgfpathcurveto{\pgfqpoint{0.259cm}{1.718cm}}{\pgfqpoint{0.273cm}{1.753cm}}{\pgfqpoint{0.273cm}{1.789cm}}
\pgfusepath{fill}
\pgfpathmoveto{\pgfqpoint{1.345cm}{1.765cm}}
\pgfpathcurveto{\pgfqpoint{1.345cm}{1.801cm}}{\pgfqpoint{1.331cm}{1.836cm}}{\pgfqpoint{1.305cm}{1.862cm}}
\pgfpathcurveto{\pgfqpoint{1.28cm}{1.887cm}}{\pgfqpoint{1.245cm}{1.902cm}}{\pgfqpoint{1.209cm}{1.902cm}}
\pgfpathcurveto{\pgfqpoint{1.172cm}{1.902cm}}{\pgfqpoint{1.138cm}{1.887cm}}{\pgfqpoint{1.112cm}{1.862cm}}
\pgfpathcurveto{\pgfqpoint{1.087cm}{1.836cm}}{\pgfqpoint{1.072cm}{1.801cm}}{\pgfqpoint{1.072cm}{1.765cm}}
\pgfpathcurveto{\pgfqpoint{1.072cm}{1.728cm}}{\pgfqpoint{1.087cm}{1.694cm}}{\pgfqpoint{1.112cm}{1.668cm}}
\pgfpathcurveto{\pgfqpoint{1.138cm}{1.642cm}}{\pgfqpoint{1.172cm}{1.628cm}}{\pgfqpoint{1.209cm}{1.628cm}}
\pgfpathcurveto{\pgfqpoint{1.245cm}{1.628cm}}{\pgfqpoint{1.28cm}{1.642cm}}{\pgfqpoint{1.305cm}{1.668cm}}
\pgfpathcurveto{\pgfqpoint{1.331cm}{1.694cm}}{\pgfqpoint{1.345cm}{1.728cm}}{\pgfqpoint{1.345cm}{1.765cm}}
\pgfusepath{fill}
\begin{pgfscope}
\pgfsetdash{}{0cm}
\pgfsetlinewidth{0.818mm}
\pgfsetroundcap
\pgfsetroundjoin
\pgfsetmiterlimit{7.0}
\pgfpathmoveto{\pgfqpoint{0.682cm}{1.065cm}}
\pgfpathlineto{\pgfqpoint{1.246cm}{0.315cm}}
\pgfpathlineto{\pgfqpoint{1.811cm}{1.065cm}}
\pgfusepath{stroke}
\end{pgfscope}
\pgfpathmoveto{\pgfqpoint{1.948cm}{1.065cm}}
\pgfpathcurveto{\pgfqpoint{1.948cm}{1.101cm}}{\pgfqpoint{1.933cm}{1.136cm}}{\pgfqpoint{1.907cm}{1.162cm}}
\pgfpathcurveto{\pgfqpoint{1.882cm}{1.187cm}}{\pgfqpoint{1.847cm}{1.202cm}}{\pgfqpoint{1.811cm}{1.202cm}}
\pgfpathcurveto{\pgfqpoint{1.775cm}{1.202cm}}{\pgfqpoint{1.74cm}{1.187cm}}{\pgfqpoint{1.714cm}{1.162cm}}
\pgfpathcurveto{\pgfqpoint{1.689cm}{1.136cm}}{\pgfqpoint{1.674cm}{1.101cm}}{\pgfqpoint{1.674cm}{1.065cm}}
\pgfpathcurveto{\pgfqpoint{1.674cm}{1.029cm}}{\pgfqpoint{1.689cm}{0.994cm}}{\pgfqpoint{1.714cm}{0.968cm}}
\pgfpathcurveto{\pgfqpoint{1.74cm}{0.942cm}}{\pgfqpoint{1.775cm}{0.928cm}}{\pgfqpoint{1.811cm}{0.928cm}}
\pgfpathcurveto{\pgfqpoint{1.847cm}{0.928cm}}{\pgfqpoint{1.882cm}{0.942cm}}{\pgfqpoint{1.907cm}{0.968cm}}
\pgfpathcurveto{\pgfqpoint{1.933cm}{0.994cm}}{\pgfqpoint{1.948cm}{1.029cm}}{\pgfqpoint{1.948cm}{1.065cm}}
\pgfusepath{fill}
\begin{pgfscope}
\pgfsetdash{}{0cm}
\pgfsetlinewidth{0.818mm}
\pgfsetmiterlimit{7.0}
\pgfpathmoveto{\pgfqpoint{1.246cm}{0.315cm}}
\pgfpathlineto{\pgfqpoint{1.244cm}{1.061cm}}
\pgfusepath{stroke}
\end{pgfscope}
\pgfpathmoveto{\pgfqpoint{1.38cm}{1.065cm}}
\pgfpathcurveto{\pgfqpoint{1.38cm}{1.101cm}}{\pgfqpoint{1.366cm}{1.136cm}}{\pgfqpoint{1.34cm}{1.162cm}}
\pgfpathcurveto{\pgfqpoint{1.315cm}{1.187cm}}{\pgfqpoint{1.28cm}{1.202cm}}{\pgfqpoint{1.244cm}{1.202cm}}
\pgfpathcurveto{\pgfqpoint{1.207cm}{1.202cm}}{\pgfqpoint{1.173cm}{1.187cm}}{\pgfqpoint{1.147cm}{1.162cm}}
\pgfpathcurveto{\pgfqpoint{1.121cm}{1.136cm}}{\pgfqpoint{1.107cm}{1.101cm}}{\pgfqpoint{1.107cm}{1.065cm}}
\pgfpathcurveto{\pgfqpoint{1.107cm}{1.029cm}}{\pgfqpoint{1.121cm}{0.994cm}}{\pgfqpoint{1.147cm}{0.968cm}}
\pgfpathcurveto{\pgfqpoint{1.173cm}{0.942cm}}{\pgfqpoint{1.207cm}{0.928cm}}{\pgfqpoint{1.244cm}{0.928cm}}
\pgfpathcurveto{\pgfqpoint{1.28cm}{0.928cm}}{\pgfqpoint{1.315cm}{0.942cm}}{\pgfqpoint{1.34cm}{0.968cm}}
\pgfpathcurveto{\pgfqpoint{1.366cm}{0.994cm}}{\pgfqpoint{1.38cm}{1.029cm}}{\pgfqpoint{1.38cm}{1.065cm}}
\pgfusepath{fill}
\begin{pgfscope}
\pgfsetdash{}{0cm}
\pgfsetlinewidth{0.818mm}
\pgfsetmiterlimit{4.0}
\pgfpathmoveto{\pgfqpoint{1.383cm}{0.178cm}}
\pgfpathcurveto{\pgfqpoint{1.383cm}{0.214cm}}{\pgfqpoint{1.369cm}{0.249cm}}{\pgfqpoint{1.343cm}{0.275cm}}
\pgfpathcurveto{\pgfqpoint{1.317cm}{0.3cm}}{\pgfqpoint{1.283cm}{0.315cm}}{\pgfqpoint{1.246cm}{0.315cm}}
\pgfpathcurveto{\pgfqpoint{1.21cm}{0.315cm}}{\pgfqpoint{1.175cm}{0.3cm}}{\pgfqpoint{1.15cm}{0.275cm}}
\pgfpathcurveto{\pgfqpoint{1.124cm}{0.249cm}}{\pgfqpoint{1.11cm}{0.214cm}}{\pgfqpoint{1.11cm}{0.178cm}}
\pgfpathcurveto{\pgfqpoint{1.11cm}{0.141cm}}{\pgfqpoint{1.124cm}{0.107cm}}{\pgfqpoint{1.15cm}{0.081cm}}
\pgfpathcurveto{\pgfqpoint{1.175cm}{0.055cm}}{\pgfqpoint{1.21cm}{0.041cm}}{\pgfqpoint{1.246cm}{0.041cm}}
\pgfpathcurveto{\pgfqpoint{1.283cm}{0.041cm}}{\pgfqpoint{1.317cm}{0.055cm}}{\pgfqpoint{1.343cm}{0.081cm}}
\pgfpathcurveto{\pgfqpoint{1.369cm}{0.107cm}}{\pgfqpoint{1.383cm}{0.141cm}}{\pgfqpoint{1.383cm}{0.178cm}}
\pgfusepath{stroke}
\end{pgfscope}
\end{pgfscope}
\end{pgfscope}
\end{pgfscope}
\end{tikzpicture}}}}\\
&\qquad +\rmb{6\UU_\leqslant X\prec(X^{\!\resizebox{0.6em}{!}{
\begin{tikzpicture}
\pgfpathmoveto{\pgfqpoint{0cm}{-0.035cm}}
\pgfpathlineto{\pgfqpoint{1.376cm}{-0.035cm}}
\pgfpathlineto{\pgfqpoint{1.376cm}{1.552cm}}
\pgfpathlineto{\pgfqpoint{0cm}{1.552cm}}
\pgfpathclose
\pgfusepath{clip}
\begin{pgfscope}
\begin{pgfscope}
\pgfpathmoveto{\pgfqpoint{0cm}{-0.035cm}}
\pgfpathlineto{\pgfqpoint{1.376cm}{-0.035cm}}
\pgfpathlineto{\pgfqpoint{1.376cm}{1.552cm}}
\pgfpathlineto{\pgfqpoint{0cm}{1.552cm}}
\pgfpathclose
\pgfusepath{clip}
\begin{pgfscope}
\begin{pgfscope}
\pgfsetdash{}{0cm}
\pgfsetlinewidth{0.818mm}
\pgfsetroundcap
\pgfsetroundjoin
\pgfsetmiterlimit{7.0}
\definecolor{eps2pgf_color}{gray}{0}\pgfsetstrokecolor{eps2pgf_color}\pgfsetfillcolor{eps2pgf_color}
\pgfpathmoveto{\pgfqpoint{0.117cm}{1.421cm}}
\pgfpathlineto{\pgfqpoint{0.682cm}{0.671cm}}
\pgfpathlineto{\pgfqpoint{1.246cm}{1.421cm}}
\pgfusepath{stroke}
\end{pgfscope}
\definecolor{eps2pgf_color}{gray}{0}\pgfsetstrokecolor{eps2pgf_color}\pgfsetfillcolor{eps2pgf_color}
\pgfpathmoveto{\pgfqpoint{0.273cm}{1.395cm}}
\pgfpathcurveto{\pgfqpoint{0.273cm}{1.432cm}}{\pgfqpoint{0.259cm}{1.467cm}}{\pgfqpoint{0.233cm}{1.492cm}}
\pgfpathcurveto{\pgfqpoint{0.207cm}{1.518cm}}{\pgfqpoint{0.173cm}{1.532cm}}{\pgfqpoint{0.137cm}{1.532cm}}
\pgfpathcurveto{\pgfqpoint{0.1cm}{1.532cm}}{\pgfqpoint{0.066cm}{1.518cm}}{\pgfqpoint{0.04cm}{1.492cm}}
\pgfpathcurveto{\pgfqpoint{0.014cm}{1.467cm}}{\pgfqpoint{0cm}{1.432cm}}{\pgfqpoint{0cm}{1.395cm}}
\pgfpathcurveto{\pgfqpoint{0cm}{1.359cm}}{\pgfqpoint{0.014cm}{1.324cm}}{\pgfqpoint{0.04cm}{1.299cm}}
\pgfpathcurveto{\pgfqpoint{0.066cm}{1.273cm}}{\pgfqpoint{0.1cm}{1.258cm}}{\pgfqpoint{0.137cm}{1.258cm}}
\pgfpathcurveto{\pgfqpoint{0.173cm}{1.258cm}}{\pgfqpoint{0.207cm}{1.273cm}}{\pgfqpoint{0.233cm}{1.299cm}}
\pgfpathcurveto{\pgfqpoint{0.259cm}{1.324cm}}{\pgfqpoint{0.273cm}{1.359cm}}{\pgfqpoint{0.273cm}{1.395cm}}
\pgfusepath{fill}
\begin{pgfscope}
\pgfsetdash{}{0cm}
\pgfsetlinewidth{0.818mm}
\pgfsetmiterlimit{7.0}
\pgfpathmoveto{\pgfqpoint{0.682cm}{0.671cm}}
\pgfpathlineto{\pgfqpoint{0.679cm}{1.418cm}}
\pgfusepath{stroke}
\end{pgfscope}
\pgfpathmoveto{\pgfqpoint{0.815cm}{1.399cm}}
\pgfpathcurveto{\pgfqpoint{0.815cm}{1.435cm}}{\pgfqpoint{0.801cm}{1.47cm}}{\pgfqpoint{0.775cm}{1.496cm}}
\pgfpathcurveto{\pgfqpoint{0.75cm}{1.521cm}}{\pgfqpoint{0.715cm}{1.536cm}}{\pgfqpoint{0.679cm}{1.536cm}}
\pgfpathcurveto{\pgfqpoint{0.643cm}{1.536cm}}{\pgfqpoint{0.608cm}{1.521cm}}{\pgfqpoint{0.582cm}{1.496cm}}
\pgfpathcurveto{\pgfqpoint{0.557cm}{1.47cm}}{\pgfqpoint{0.542cm}{1.435cm}}{\pgfqpoint{0.542cm}{1.399cm}}
\pgfpathcurveto{\pgfqpoint{0.542cm}{1.363cm}}{\pgfqpoint{0.557cm}{1.328cm}}{\pgfqpoint{0.582cm}{1.302cm}}
\pgfpathcurveto{\pgfqpoint{0.608cm}{1.276cm}}{\pgfqpoint{0.643cm}{1.262cm}}{\pgfqpoint{0.679cm}{1.262cm}}
\pgfpathcurveto{\pgfqpoint{0.715cm}{1.262cm}}{\pgfqpoint{0.75cm}{1.276cm}}{\pgfqpoint{0.775cm}{1.302cm}}
\pgfpathcurveto{\pgfqpoint{0.801cm}{1.328cm}}{\pgfqpoint{0.815cm}{1.363cm}}{\pgfqpoint{0.815cm}{1.399cm}}
\pgfusepath{fill}
\pgfpathmoveto{\pgfqpoint{1.345cm}{1.371cm}}
\pgfpathcurveto{\pgfqpoint{1.345cm}{1.408cm}}{\pgfqpoint{1.331cm}{1.442cm}}{\pgfqpoint{1.305cm}{1.468cm}}
\pgfpathcurveto{\pgfqpoint{1.28cm}{1.494cm}}{\pgfqpoint{1.245cm}{1.508cm}}{\pgfqpoint{1.209cm}{1.508cm}}
\pgfpathcurveto{\pgfqpoint{1.172cm}{1.508cm}}{\pgfqpoint{1.138cm}{1.494cm}}{\pgfqpoint{1.112cm}{1.468cm}}
\pgfpathcurveto{\pgfqpoint{1.087cm}{1.442cm}}{\pgfqpoint{1.072cm}{1.408cm}}{\pgfqpoint{1.072cm}{1.371cm}}
\pgfpathcurveto{\pgfqpoint{1.072cm}{1.335cm}}{\pgfqpoint{1.087cm}{1.3cm}}{\pgfqpoint{1.112cm}{1.274cm}}
\pgfpathcurveto{\pgfqpoint{1.138cm}{1.249cm}}{\pgfqpoint{1.172cm}{1.234cm}}{\pgfqpoint{1.209cm}{1.234cm}}
\pgfpathcurveto{\pgfqpoint{1.245cm}{1.234cm}}{\pgfqpoint{1.28cm}{1.249cm}}{\pgfqpoint{1.305cm}{1.274cm}}
\pgfpathcurveto{\pgfqpoint{1.331cm}{1.3cm}}{\pgfqpoint{1.345cm}{1.335cm}}{\pgfqpoint{1.345cm}{1.371cm}}
\pgfusepath{fill}
\begin{pgfscope}
\pgfsetdash{}{0cm}
\pgfsetlinewidth{0.818mm}
\pgfsetroundcap
\pgfsetmiterlimit{4.0}
\pgfpathmoveto{\pgfqpoint{0.682cm}{0.671cm}}
\pgfpathlineto{\pgfqpoint{0.682cm}{0.042cm}}
\pgfusepath{stroke}
\end{pgfscope}
\end{pgfscope}
\end{pgfscope}
\end{pgfscope}
\end{tikzpicture}}}(\phi+\psi))}+\rmb{6\UU_\leqslant X\prec(X^{\!\resizebox{0.6em}{!}{
\begin{tikzpicture}
\pgfpathmoveto{\pgfqpoint{0cm}{-0.035cm}}
\pgfpathlineto{\pgfqpoint{1.376cm}{-0.035cm}}
\pgfpathlineto{\pgfqpoint{1.376cm}{1.552cm}}
\pgfpathlineto{\pgfqpoint{0cm}{1.552cm}}
\pgfpathclose
\pgfusepath{clip}
\begin{pgfscope}
\begin{pgfscope}
\pgfpathmoveto{\pgfqpoint{0cm}{-0.035cm}}
\pgfpathlineto{\pgfqpoint{1.376cm}{-0.035cm}}
\pgfpathlineto{\pgfqpoint{1.376cm}{1.552cm}}
\pgfpathlineto{\pgfqpoint{0cm}{1.552cm}}
\pgfpathclose
\pgfusepath{clip}
\begin{pgfscope}
\begin{pgfscope}
\pgfsetdash{}{0cm}
\pgfsetlinewidth{0.818mm}
\pgfsetroundcap
\pgfsetroundjoin
\pgfsetmiterlimit{7.0}
\definecolor{eps2pgf_color}{gray}{0}\pgfsetstrokecolor{eps2pgf_color}\pgfsetfillcolor{eps2pgf_color}
\pgfpathmoveto{\pgfqpoint{0.117cm}{1.421cm}}
\pgfpathlineto{\pgfqpoint{0.682cm}{0.671cm}}
\pgfpathlineto{\pgfqpoint{1.246cm}{1.421cm}}
\pgfusepath{stroke}
\end{pgfscope}
\definecolor{eps2pgf_color}{gray}{0}\pgfsetstrokecolor{eps2pgf_color}\pgfsetfillcolor{eps2pgf_color}
\pgfpathmoveto{\pgfqpoint{0.273cm}{1.395cm}}
\pgfpathcurveto{\pgfqpoint{0.273cm}{1.432cm}}{\pgfqpoint{0.259cm}{1.467cm}}{\pgfqpoint{0.233cm}{1.492cm}}
\pgfpathcurveto{\pgfqpoint{0.207cm}{1.518cm}}{\pgfqpoint{0.173cm}{1.532cm}}{\pgfqpoint{0.137cm}{1.532cm}}
\pgfpathcurveto{\pgfqpoint{0.1cm}{1.532cm}}{\pgfqpoint{0.066cm}{1.518cm}}{\pgfqpoint{0.04cm}{1.492cm}}
\pgfpathcurveto{\pgfqpoint{0.014cm}{1.467cm}}{\pgfqpoint{0cm}{1.432cm}}{\pgfqpoint{0cm}{1.395cm}}
\pgfpathcurveto{\pgfqpoint{0cm}{1.359cm}}{\pgfqpoint{0.014cm}{1.324cm}}{\pgfqpoint{0.04cm}{1.299cm}}
\pgfpathcurveto{\pgfqpoint{0.066cm}{1.273cm}}{\pgfqpoint{0.1cm}{1.258cm}}{\pgfqpoint{0.137cm}{1.258cm}}
\pgfpathcurveto{\pgfqpoint{0.173cm}{1.258cm}}{\pgfqpoint{0.207cm}{1.273cm}}{\pgfqpoint{0.233cm}{1.299cm}}
\pgfpathcurveto{\pgfqpoint{0.259cm}{1.324cm}}{\pgfqpoint{0.273cm}{1.359cm}}{\pgfqpoint{0.273cm}{1.395cm}}
\pgfusepath{fill}
\begin{pgfscope}
\pgfsetdash{}{0cm}
\pgfsetlinewidth{0.818mm}
\pgfsetmiterlimit{7.0}
\pgfpathmoveto{\pgfqpoint{0.682cm}{0.671cm}}
\pgfpathlineto{\pgfqpoint{0.679cm}{1.418cm}}
\pgfusepath{stroke}
\end{pgfscope}
\pgfpathmoveto{\pgfqpoint{0.815cm}{1.399cm}}
\pgfpathcurveto{\pgfqpoint{0.815cm}{1.435cm}}{\pgfqpoint{0.801cm}{1.47cm}}{\pgfqpoint{0.775cm}{1.496cm}}
\pgfpathcurveto{\pgfqpoint{0.75cm}{1.521cm}}{\pgfqpoint{0.715cm}{1.536cm}}{\pgfqpoint{0.679cm}{1.536cm}}
\pgfpathcurveto{\pgfqpoint{0.643cm}{1.536cm}}{\pgfqpoint{0.608cm}{1.521cm}}{\pgfqpoint{0.582cm}{1.496cm}}
\pgfpathcurveto{\pgfqpoint{0.557cm}{1.47cm}}{\pgfqpoint{0.542cm}{1.435cm}}{\pgfqpoint{0.542cm}{1.399cm}}
\pgfpathcurveto{\pgfqpoint{0.542cm}{1.363cm}}{\pgfqpoint{0.557cm}{1.328cm}}{\pgfqpoint{0.582cm}{1.302cm}}
\pgfpathcurveto{\pgfqpoint{0.608cm}{1.276cm}}{\pgfqpoint{0.643cm}{1.262cm}}{\pgfqpoint{0.679cm}{1.262cm}}
\pgfpathcurveto{\pgfqpoint{0.715cm}{1.262cm}}{\pgfqpoint{0.75cm}{1.276cm}}{\pgfqpoint{0.775cm}{1.302cm}}
\pgfpathcurveto{\pgfqpoint{0.801cm}{1.328cm}}{\pgfqpoint{0.815cm}{1.363cm}}{\pgfqpoint{0.815cm}{1.399cm}}
\pgfusepath{fill}
\pgfpathmoveto{\pgfqpoint{1.345cm}{1.371cm}}
\pgfpathcurveto{\pgfqpoint{1.345cm}{1.408cm}}{\pgfqpoint{1.331cm}{1.442cm}}{\pgfqpoint{1.305cm}{1.468cm}}
\pgfpathcurveto{\pgfqpoint{1.28cm}{1.494cm}}{\pgfqpoint{1.245cm}{1.508cm}}{\pgfqpoint{1.209cm}{1.508cm}}
\pgfpathcurveto{\pgfqpoint{1.172cm}{1.508cm}}{\pgfqpoint{1.138cm}{1.494cm}}{\pgfqpoint{1.112cm}{1.468cm}}
\pgfpathcurveto{\pgfqpoint{1.087cm}{1.442cm}}{\pgfqpoint{1.072cm}{1.408cm}}{\pgfqpoint{1.072cm}{1.371cm}}
\pgfpathcurveto{\pgfqpoint{1.072cm}{1.335cm}}{\pgfqpoint{1.087cm}{1.3cm}}{\pgfqpoint{1.112cm}{1.274cm}}
\pgfpathcurveto{\pgfqpoint{1.138cm}{1.249cm}}{\pgfqpoint{1.172cm}{1.234cm}}{\pgfqpoint{1.209cm}{1.234cm}}
\pgfpathcurveto{\pgfqpoint{1.245cm}{1.234cm}}{\pgfqpoint{1.28cm}{1.249cm}}{\pgfqpoint{1.305cm}{1.274cm}}
\pgfpathcurveto{\pgfqpoint{1.331cm}{1.3cm}}{\pgfqpoint{1.345cm}{1.335cm}}{\pgfqpoint{1.345cm}{1.371cm}}
\pgfusepath{fill}
\begin{pgfscope}
\pgfsetdash{}{0cm}
\pgfsetlinewidth{0.818mm}
\pgfsetroundcap
\pgfsetmiterlimit{4.0}
\pgfpathmoveto{\pgfqpoint{0.682cm}{0.671cm}}
\pgfpathlineto{\pgfqpoint{0.682cm}{0.042cm}}
\pgfusepath{stroke}
\end{pgfscope}
\end{pgfscope}
\end{pgfscope}
\end{pgfscope}
\end{tikzpicture}}}(\phi+\psi))}\\
&\qquad+\rmb{6(\phi+\psi)\prec\UU_{\leqslant}X^{\!\resizebox{!}{.8em}{
\begin{tikzpicture}
\pgfpathmoveto{\pgfqpoint{0cm}{-0.035cm}}
\pgfpathlineto{\pgfqpoint{1.976cm}{-0.035cm}}
\pgfpathlineto{\pgfqpoint{1.976cm}{1.94cm}}
\pgfpathlineto{\pgfqpoint{0cm}{1.94cm}}
\pgfpathclose
\pgfusepath{clip}
\begin{pgfscope}
\begin{pgfscope}
\pgfpathmoveto{\pgfqpoint{0cm}{-0.035cm}}
\pgfpathlineto{\pgfqpoint{1.976cm}{-0.035cm}}
\pgfpathlineto{\pgfqpoint{1.976cm}{1.94cm}}
\pgfpathlineto{\pgfqpoint{0cm}{1.94cm}}
\pgfpathclose
\pgfusepath{clip}
\begin{pgfscope}
\begin{pgfscope}
\pgfsetdash{}{0cm}
\pgfsetlinewidth{0.818mm}
\pgfsetroundcap
\pgfsetroundjoin
\pgfsetmiterlimit{7.0}
\definecolor{eps2pgf_color}{gray}{0}\pgfsetstrokecolor{eps2pgf_color}\pgfsetfillcolor{eps2pgf_color}
\pgfpathmoveto{\pgfqpoint{0.117cm}{1.815cm}}
\pgfpathlineto{\pgfqpoint{0.682cm}{1.065cm}}
\pgfpathlineto{\pgfqpoint{1.246cm}{1.815cm}}
\pgfusepath{stroke}
\end{pgfscope}
\definecolor{eps2pgf_color}{gray}{0}\pgfsetstrokecolor{eps2pgf_color}\pgfsetfillcolor{eps2pgf_color}
\pgfpathmoveto{\pgfqpoint{0.273cm}{1.789cm}}
\pgfpathcurveto{\pgfqpoint{0.273cm}{1.825cm}}{\pgfqpoint{0.259cm}{1.86cm}}{\pgfqpoint{0.233cm}{1.886cm}}
\pgfpathcurveto{\pgfqpoint{0.207cm}{1.912cm}}{\pgfqpoint{0.173cm}{1.926cm}}{\pgfqpoint{0.137cm}{1.926cm}}
\pgfpathcurveto{\pgfqpoint{0.1cm}{1.926cm}}{\pgfqpoint{0.066cm}{1.912cm}}{\pgfqpoint{0.04cm}{1.886cm}}
\pgfpathcurveto{\pgfqpoint{0.014cm}{1.86cm}}{\pgfqpoint{0cm}{1.825cm}}{\pgfqpoint{0cm}{1.789cm}}
\pgfpathcurveto{\pgfqpoint{0cm}{1.753cm}}{\pgfqpoint{0.014cm}{1.718cm}}{\pgfqpoint{0.04cm}{1.692cm}}
\pgfpathcurveto{\pgfqpoint{0.066cm}{1.667cm}}{\pgfqpoint{0.1cm}{1.652cm}}{\pgfqpoint{0.137cm}{1.652cm}}
\pgfpathcurveto{\pgfqpoint{0.173cm}{1.652cm}}{\pgfqpoint{0.207cm}{1.667cm}}{\pgfqpoint{0.233cm}{1.692cm}}
\pgfpathcurveto{\pgfqpoint{0.259cm}{1.718cm}}{\pgfqpoint{0.273cm}{1.753cm}}{\pgfqpoint{0.273cm}{1.789cm}}
\pgfusepath{fill}
\begin{pgfscope}
\pgfsetdash{}{0cm}
\pgfsetlinewidth{0.818mm}
\pgfsetmiterlimit{7.0}
\pgfpathmoveto{\pgfqpoint{0.682cm}{1.065cm}}
\pgfpathlineto{\pgfqpoint{0.679cm}{1.812cm}}
\pgfusepath{stroke}
\end{pgfscope}
\pgfpathmoveto{\pgfqpoint{0.815cm}{1.793cm}}
\pgfpathcurveto{\pgfqpoint{0.815cm}{1.829cm}}{\pgfqpoint{0.801cm}{1.864cm}}{\pgfqpoint{0.775cm}{1.89cm}}
\pgfpathcurveto{\pgfqpoint{0.75cm}{1.915cm}}{\pgfqpoint{0.715cm}{1.93cm}}{\pgfqpoint{0.679cm}{1.93cm}}
\pgfpathcurveto{\pgfqpoint{0.643cm}{1.93cm}}{\pgfqpoint{0.608cm}{1.915cm}}{\pgfqpoint{0.582cm}{1.89cm}}
\pgfpathcurveto{\pgfqpoint{0.557cm}{1.864cm}}{\pgfqpoint{0.542cm}{1.829cm}}{\pgfqpoint{0.542cm}{1.793cm}}
\pgfpathcurveto{\pgfqpoint{0.542cm}{1.756cm}}{\pgfqpoint{0.557cm}{1.722cm}}{\pgfqpoint{0.582cm}{1.696cm}}
\pgfpathcurveto{\pgfqpoint{0.608cm}{1.67cm}}{\pgfqpoint{0.643cm}{1.656cm}}{\pgfqpoint{0.679cm}{1.656cm}}
\pgfpathcurveto{\pgfqpoint{0.715cm}{1.656cm}}{\pgfqpoint{0.75cm}{1.67cm}}{\pgfqpoint{0.775cm}{1.696cm}}
\pgfpathcurveto{\pgfqpoint{0.801cm}{1.722cm}}{\pgfqpoint{0.815cm}{1.756cm}}{\pgfqpoint{0.815cm}{1.793cm}}
\pgfusepath{fill}
\pgfpathmoveto{\pgfqpoint{1.345cm}{1.765cm}}
\pgfpathcurveto{\pgfqpoint{1.345cm}{1.801cm}}{\pgfqpoint{1.331cm}{1.836cm}}{\pgfqpoint{1.305cm}{1.862cm}}
\pgfpathcurveto{\pgfqpoint{1.28cm}{1.887cm}}{\pgfqpoint{1.245cm}{1.902cm}}{\pgfqpoint{1.209cm}{1.902cm}}
\pgfpathcurveto{\pgfqpoint{1.172cm}{1.902cm}}{\pgfqpoint{1.138cm}{1.887cm}}{\pgfqpoint{1.112cm}{1.862cm}}
\pgfpathcurveto{\pgfqpoint{1.087cm}{1.836cm}}{\pgfqpoint{1.072cm}{1.801cm}}{\pgfqpoint{1.072cm}{1.765cm}}
\pgfpathcurveto{\pgfqpoint{1.072cm}{1.728cm}}{\pgfqpoint{1.087cm}{1.694cm}}{\pgfqpoint{1.112cm}{1.668cm}}
\pgfpathcurveto{\pgfqpoint{1.138cm}{1.642cm}}{\pgfqpoint{1.172cm}{1.628cm}}{\pgfqpoint{1.209cm}{1.628cm}}
\pgfpathcurveto{\pgfqpoint{1.245cm}{1.628cm}}{\pgfqpoint{1.28cm}{1.642cm}}{\pgfqpoint{1.305cm}{1.668cm}}
\pgfpathcurveto{\pgfqpoint{1.331cm}{1.694cm}}{\pgfqpoint{1.345cm}{1.728cm}}{\pgfqpoint{1.345cm}{1.765cm}}
\pgfusepath{fill}
\begin{pgfscope}
\pgfsetdash{}{0cm}
\pgfsetlinewidth{0.818mm}
\pgfsetroundcap
\pgfsetroundjoin
\pgfsetmiterlimit{7.0}
\pgfpathmoveto{\pgfqpoint{0.682cm}{1.065cm}}
\pgfpathlineto{\pgfqpoint{1.246cm}{0.315cm}}
\pgfpathlineto{\pgfqpoint{1.811cm}{1.065cm}}
\pgfusepath{stroke}
\end{pgfscope}
\pgfpathmoveto{\pgfqpoint{1.948cm}{1.065cm}}
\pgfpathcurveto{\pgfqpoint{1.948cm}{1.101cm}}{\pgfqpoint{1.933cm}{1.136cm}}{\pgfqpoint{1.907cm}{1.162cm}}
\pgfpathcurveto{\pgfqpoint{1.882cm}{1.187cm}}{\pgfqpoint{1.847cm}{1.202cm}}{\pgfqpoint{1.811cm}{1.202cm}}
\pgfpathcurveto{\pgfqpoint{1.775cm}{1.202cm}}{\pgfqpoint{1.74cm}{1.187cm}}{\pgfqpoint{1.714cm}{1.162cm}}
\pgfpathcurveto{\pgfqpoint{1.689cm}{1.136cm}}{\pgfqpoint{1.674cm}{1.101cm}}{\pgfqpoint{1.674cm}{1.065cm}}
\pgfpathcurveto{\pgfqpoint{1.674cm}{1.029cm}}{\pgfqpoint{1.689cm}{0.994cm}}{\pgfqpoint{1.714cm}{0.968cm}}
\pgfpathcurveto{\pgfqpoint{1.74cm}{0.942cm}}{\pgfqpoint{1.775cm}{0.928cm}}{\pgfqpoint{1.811cm}{0.928cm}}
\pgfpathcurveto{\pgfqpoint{1.847cm}{0.928cm}}{\pgfqpoint{1.882cm}{0.942cm}}{\pgfqpoint{1.907cm}{0.968cm}}
\pgfpathcurveto{\pgfqpoint{1.933cm}{0.994cm}}{\pgfqpoint{1.948cm}{1.029cm}}{\pgfqpoint{1.948cm}{1.065cm}}
\pgfusepath{fill}
\begin{pgfscope}
\pgfsetdash{}{0cm}
\pgfsetlinewidth{0.818mm}
\pgfsetmiterlimit{4.0}
\pgfpathmoveto{\pgfqpoint{1.383cm}{0.178cm}}
\pgfpathcurveto{\pgfqpoint{1.383cm}{0.214cm}}{\pgfqpoint{1.369cm}{0.249cm}}{\pgfqpoint{1.343cm}{0.275cm}}
\pgfpathcurveto{\pgfqpoint{1.317cm}{0.3cm}}{\pgfqpoint{1.283cm}{0.315cm}}{\pgfqpoint{1.246cm}{0.315cm}}
\pgfpathcurveto{\pgfqpoint{1.21cm}{0.315cm}}{\pgfqpoint{1.175cm}{0.3cm}}{\pgfqpoint{1.15cm}{0.275cm}}
\pgfpathcurveto{\pgfqpoint{1.124cm}{0.249cm}}{\pgfqpoint{1.11cm}{0.214cm}}{\pgfqpoint{1.11cm}{0.178cm}}
\pgfpathcurveto{\pgfqpoint{1.11cm}{0.141cm}}{\pgfqpoint{1.124cm}{0.107cm}}{\pgfqpoint{1.15cm}{0.081cm}}
\pgfpathcurveto{\pgfqpoint{1.175cm}{0.055cm}}{\pgfqpoint{1.21cm}{0.041cm}}{\pgfqpoint{1.246cm}{0.041cm}}
\pgfpathcurveto{\pgfqpoint{1.283cm}{0.041cm}}{\pgfqpoint{1.317cm}{0.055cm}}{\pgfqpoint{1.343cm}{0.081cm}}
\pgfpathcurveto{\pgfqpoint{1.369cm}{0.107cm}}{\pgfqpoint{1.383cm}{0.141cm}}{\pgfqpoint{1.383cm}{0.178cm}}
\pgfusepath{stroke}
\end{pgfscope}
\end{pgfscope}
\end{pgfscope}
\end{pgfscope}
\end{tikzpicture}}}}+\rmb{6(\phi+\psi)\succcurlyeqX^{\!\resizebox{!}{.8em}{
\begin{tikzpicture}
\pgfpathmoveto{\pgfqpoint{0cm}{-0.035cm}}
\pgfpathlineto{\pgfqpoint{1.976cm}{-0.035cm}}
\pgfpathlineto{\pgfqpoint{1.976cm}{1.94cm}}
\pgfpathlineto{\pgfqpoint{0cm}{1.94cm}}
\pgfpathclose
\pgfusepath{clip}
\begin{pgfscope}
\begin{pgfscope}
\pgfpathmoveto{\pgfqpoint{0cm}{-0.035cm}}
\pgfpathlineto{\pgfqpoint{1.976cm}{-0.035cm}}
\pgfpathlineto{\pgfqpoint{1.976cm}{1.94cm}}
\pgfpathlineto{\pgfqpoint{0cm}{1.94cm}}
\pgfpathclose
\pgfusepath{clip}
\begin{pgfscope}
\begin{pgfscope}
\pgfsetdash{}{0cm}
\pgfsetlinewidth{0.818mm}
\pgfsetroundcap
\pgfsetroundjoin
\pgfsetmiterlimit{7.0}
\definecolor{eps2pgf_color}{gray}{0}\pgfsetstrokecolor{eps2pgf_color}\pgfsetfillcolor{eps2pgf_color}
\pgfpathmoveto{\pgfqpoint{0.117cm}{1.815cm}}
\pgfpathlineto{\pgfqpoint{0.682cm}{1.065cm}}
\pgfpathlineto{\pgfqpoint{1.246cm}{1.815cm}}
\pgfusepath{stroke}
\end{pgfscope}
\definecolor{eps2pgf_color}{gray}{0}\pgfsetstrokecolor{eps2pgf_color}\pgfsetfillcolor{eps2pgf_color}
\pgfpathmoveto{\pgfqpoint{0.273cm}{1.789cm}}
\pgfpathcurveto{\pgfqpoint{0.273cm}{1.825cm}}{\pgfqpoint{0.259cm}{1.86cm}}{\pgfqpoint{0.233cm}{1.886cm}}
\pgfpathcurveto{\pgfqpoint{0.207cm}{1.912cm}}{\pgfqpoint{0.173cm}{1.926cm}}{\pgfqpoint{0.137cm}{1.926cm}}
\pgfpathcurveto{\pgfqpoint{0.1cm}{1.926cm}}{\pgfqpoint{0.066cm}{1.912cm}}{\pgfqpoint{0.04cm}{1.886cm}}
\pgfpathcurveto{\pgfqpoint{0.014cm}{1.86cm}}{\pgfqpoint{0cm}{1.825cm}}{\pgfqpoint{0cm}{1.789cm}}
\pgfpathcurveto{\pgfqpoint{0cm}{1.753cm}}{\pgfqpoint{0.014cm}{1.718cm}}{\pgfqpoint{0.04cm}{1.692cm}}
\pgfpathcurveto{\pgfqpoint{0.066cm}{1.667cm}}{\pgfqpoint{0.1cm}{1.652cm}}{\pgfqpoint{0.137cm}{1.652cm}}
\pgfpathcurveto{\pgfqpoint{0.173cm}{1.652cm}}{\pgfqpoint{0.207cm}{1.667cm}}{\pgfqpoint{0.233cm}{1.692cm}}
\pgfpathcurveto{\pgfqpoint{0.259cm}{1.718cm}}{\pgfqpoint{0.273cm}{1.753cm}}{\pgfqpoint{0.273cm}{1.789cm}}
\pgfusepath{fill}
\begin{pgfscope}
\pgfsetdash{}{0cm}
\pgfsetlinewidth{0.818mm}
\pgfsetmiterlimit{7.0}
\pgfpathmoveto{\pgfqpoint{0.682cm}{1.065cm}}
\pgfpathlineto{\pgfqpoint{0.679cm}{1.812cm}}
\pgfusepath{stroke}
\end{pgfscope}
\pgfpathmoveto{\pgfqpoint{0.815cm}{1.793cm}}
\pgfpathcurveto{\pgfqpoint{0.815cm}{1.829cm}}{\pgfqpoint{0.801cm}{1.864cm}}{\pgfqpoint{0.775cm}{1.89cm}}
\pgfpathcurveto{\pgfqpoint{0.75cm}{1.915cm}}{\pgfqpoint{0.715cm}{1.93cm}}{\pgfqpoint{0.679cm}{1.93cm}}
\pgfpathcurveto{\pgfqpoint{0.643cm}{1.93cm}}{\pgfqpoint{0.608cm}{1.915cm}}{\pgfqpoint{0.582cm}{1.89cm}}
\pgfpathcurveto{\pgfqpoint{0.557cm}{1.864cm}}{\pgfqpoint{0.542cm}{1.829cm}}{\pgfqpoint{0.542cm}{1.793cm}}
\pgfpathcurveto{\pgfqpoint{0.542cm}{1.756cm}}{\pgfqpoint{0.557cm}{1.722cm}}{\pgfqpoint{0.582cm}{1.696cm}}
\pgfpathcurveto{\pgfqpoint{0.608cm}{1.67cm}}{\pgfqpoint{0.643cm}{1.656cm}}{\pgfqpoint{0.679cm}{1.656cm}}
\pgfpathcurveto{\pgfqpoint{0.715cm}{1.656cm}}{\pgfqpoint{0.75cm}{1.67cm}}{\pgfqpoint{0.775cm}{1.696cm}}
\pgfpathcurveto{\pgfqpoint{0.801cm}{1.722cm}}{\pgfqpoint{0.815cm}{1.756cm}}{\pgfqpoint{0.815cm}{1.793cm}}
\pgfusepath{fill}
\pgfpathmoveto{\pgfqpoint{1.345cm}{1.765cm}}
\pgfpathcurveto{\pgfqpoint{1.345cm}{1.801cm}}{\pgfqpoint{1.331cm}{1.836cm}}{\pgfqpoint{1.305cm}{1.862cm}}
\pgfpathcurveto{\pgfqpoint{1.28cm}{1.887cm}}{\pgfqpoint{1.245cm}{1.902cm}}{\pgfqpoint{1.209cm}{1.902cm}}
\pgfpathcurveto{\pgfqpoint{1.172cm}{1.902cm}}{\pgfqpoint{1.138cm}{1.887cm}}{\pgfqpoint{1.112cm}{1.862cm}}
\pgfpathcurveto{\pgfqpoint{1.087cm}{1.836cm}}{\pgfqpoint{1.072cm}{1.801cm}}{\pgfqpoint{1.072cm}{1.765cm}}
\pgfpathcurveto{\pgfqpoint{1.072cm}{1.728cm}}{\pgfqpoint{1.087cm}{1.694cm}}{\pgfqpoint{1.112cm}{1.668cm}}
\pgfpathcurveto{\pgfqpoint{1.138cm}{1.642cm}}{\pgfqpoint{1.172cm}{1.628cm}}{\pgfqpoint{1.209cm}{1.628cm}}
\pgfpathcurveto{\pgfqpoint{1.245cm}{1.628cm}}{\pgfqpoint{1.28cm}{1.642cm}}{\pgfqpoint{1.305cm}{1.668cm}}
\pgfpathcurveto{\pgfqpoint{1.331cm}{1.694cm}}{\pgfqpoint{1.345cm}{1.728cm}}{\pgfqpoint{1.345cm}{1.765cm}}
\pgfusepath{fill}
\begin{pgfscope}
\pgfsetdash{}{0cm}
\pgfsetlinewidth{0.818mm}
\pgfsetroundcap
\pgfsetroundjoin
\pgfsetmiterlimit{7.0}
\pgfpathmoveto{\pgfqpoint{0.682cm}{1.065cm}}
\pgfpathlineto{\pgfqpoint{1.246cm}{0.315cm}}
\pgfpathlineto{\pgfqpoint{1.811cm}{1.065cm}}
\pgfusepath{stroke}
\end{pgfscope}
\pgfpathmoveto{\pgfqpoint{1.948cm}{1.065cm}}
\pgfpathcurveto{\pgfqpoint{1.948cm}{1.101cm}}{\pgfqpoint{1.933cm}{1.136cm}}{\pgfqpoint{1.907cm}{1.162cm}}
\pgfpathcurveto{\pgfqpoint{1.882cm}{1.187cm}}{\pgfqpoint{1.847cm}{1.202cm}}{\pgfqpoint{1.811cm}{1.202cm}}
\pgfpathcurveto{\pgfqpoint{1.775cm}{1.202cm}}{\pgfqpoint{1.74cm}{1.187cm}}{\pgfqpoint{1.714cm}{1.162cm}}
\pgfpathcurveto{\pgfqpoint{1.689cm}{1.136cm}}{\pgfqpoint{1.674cm}{1.101cm}}{\pgfqpoint{1.674cm}{1.065cm}}
\pgfpathcurveto{\pgfqpoint{1.674cm}{1.029cm}}{\pgfqpoint{1.689cm}{0.994cm}}{\pgfqpoint{1.714cm}{0.968cm}}
\pgfpathcurveto{\pgfqpoint{1.74cm}{0.942cm}}{\pgfqpoint{1.775cm}{0.928cm}}{\pgfqpoint{1.811cm}{0.928cm}}
\pgfpathcurveto{\pgfqpoint{1.847cm}{0.928cm}}{\pgfqpoint{1.882cm}{0.942cm}}{\pgfqpoint{1.907cm}{0.968cm}}
\pgfpathcurveto{\pgfqpoint{1.933cm}{0.994cm}}{\pgfqpoint{1.948cm}{1.029cm}}{\pgfqpoint{1.948cm}{1.065cm}}
\pgfusepath{fill}
\begin{pgfscope}
\pgfsetdash{}{0cm}
\pgfsetlinewidth{0.818mm}
\pgfsetmiterlimit{4.0}
\pgfpathmoveto{\pgfqpoint{1.383cm}{0.178cm}}
\pgfpathcurveto{\pgfqpoint{1.383cm}{0.214cm}}{\pgfqpoint{1.369cm}{0.249cm}}{\pgfqpoint{1.343cm}{0.275cm}}
\pgfpathcurveto{\pgfqpoint{1.317cm}{0.3cm}}{\pgfqpoint{1.283cm}{0.315cm}}{\pgfqpoint{1.246cm}{0.315cm}}
\pgfpathcurveto{\pgfqpoint{1.21cm}{0.315cm}}{\pgfqpoint{1.175cm}{0.3cm}}{\pgfqpoint{1.15cm}{0.275cm}}
\pgfpathcurveto{\pgfqpoint{1.124cm}{0.249cm}}{\pgfqpoint{1.11cm}{0.214cm}}{\pgfqpoint{1.11cm}{0.178cm}}
\pgfpathcurveto{\pgfqpoint{1.11cm}{0.141cm}}{\pgfqpoint{1.124cm}{0.107cm}}{\pgfqpoint{1.15cm}{0.081cm}}
\pgfpathcurveto{\pgfqpoint{1.175cm}{0.055cm}}{\pgfqpoint{1.21cm}{0.041cm}}{\pgfqpoint{1.246cm}{0.041cm}}
\pgfpathcurveto{\pgfqpoint{1.283cm}{0.041cm}}{\pgfqpoint{1.317cm}{0.055cm}}{\pgfqpoint{1.343cm}{0.081cm}}
\pgfpathcurveto{\pgfqpoint{1.369cm}{0.107cm}}{\pgfqpoint{1.383cm}{0.141cm}}{\pgfqpoint{1.383cm}{0.178cm}}
\pgfusepath{stroke}
\end{pgfscope}
\end{pgfscope}
\end{pgfscope}
\end{pgfscope}
\end{tikzpicture}}}}+\rmb{3\UU_\leqslant X\succ (\phi+\psi)^2}+\rmb{3 X\preccurlyeq(\phi+\psi)^2}.
\end{align*}
We require that, separately,
\begin{align}\label{eq:two}
 \Q \phi + \rmm{\Phi} = 0, \qquad \Q \psi + \psi^3 + \rmb{\Psi} = 0. 
 \end{align}
Note that in order to have the term $\llbracket X^2\rrbracket\circ\vartheta$ well-defined, it is necessary that $\vartheta$ is at least of regularity $1+\alpha$ for some $\alpha>0.$  This will be shown below.

\subsection{Bound for $\phi$ in $\CC^\alpha(\rho)$}
\label{ssec:phi1}

At this point we only consider the equation for $\phi$ and intend to show that it belongs to $\CC^{\alpha}(\rho)$ for some $\alpha>0$.
Therefore we aim to estimate $\rmm{\Phi}$ in $\CC^{-2+\alpha}(\rho)$. Recall that before we included the localization operators $\UU_>,$ $\UU_\leqslant$ above, all the magenta and all the orange terms were actually better, namely, of regularity at least $-1-\kappa$. Thanks to the operator $\UU_>$ we are able to profit from this difference of actual and wanted regularity. More precisely, we gain a small factor in all the terms in $\rmm{\Phi}$ containing $\phi+\psi$. As a consequence, a suitable choice of the parameter $L$ in the construction of $\UU_>,$ $\UU_\leqslant,$ yields  a bound for $\rmm{\Phi}$ that only depends on the data of the problem but not on the solution. 

\rmbb{Fix a parameter $K>0$ which will be chosen at the end of this subsection depending on the $L^{\infty}$-norm of $\phi+\psi$. As the next step, given $K$, we shall determine the precise value of $L$ for each application of the localization operators $\UU_{\leq},\UU_{>}$.} First of all, we observe that all the magenta terms that do not contain $\phi+\psi$ can be bounded in $\CC^{-1-\kappa}(\rho^\sigma)$. For the remaining terms, it holds
\begin{align*}
&\|3\UU_>\llbracket X^2 \rrbracket\succ(\phi+\psi)\|_{\CC^{-2+\alpha}(\rho)}+\|3\UU_>\llbracket X^2 \rrbracket\prec(\phi+\psi)\|_{\CC^{-2+\alpha}(\rho)}\\&\qquad\lesssim \|\phi+\psi\|_{L^\infty(\rho)}\|\UU_>\llbracket X^2 \rrbracket\|_{\CC^{-2+\alpha}}\lesssim 2^{-(1-\alpha-\kappa)K} \|\phi+\psi\|_{L^\infty(\rho)},
\end{align*}
provided
\begin{equation}\label{eq:l1}
\|\UU_> \llbracket X^2 \rrbracket\|_{\CC^{-2+\alpha}}\lesssim 2^{-(1-\alpha-\kappa)K} \|\llbracket X^2 \rrbracket\|_{\CC^{-1-\kappa}(\rho^\sigma)}.
\end{equation}
Similarly,
\begin{align*}
\|3( \phi + \psi)\prec\UU_>X^{\!\resizebox{!}{.8em}{
\begin{tikzpicture}
\pgfpathmoveto{\pgfqpoint{0cm}{-0.035cm}}
\pgfpathlineto{\pgfqpoint{1.976cm}{-0.035cm}}
\pgfpathlineto{\pgfqpoint{1.976cm}{1.94cm}}
\pgfpathlineto{\pgfqpoint{0cm}{1.94cm}}
\pgfpathclose
\pgfusepath{clip}
\begin{pgfscope}
\begin{pgfscope}
\pgfpathmoveto{\pgfqpoint{0cm}{-0.035cm}}
\pgfpathlineto{\pgfqpoint{1.976cm}{-0.035cm}}
\pgfpathlineto{\pgfqpoint{1.976cm}{1.94cm}}
\pgfpathlineto{\pgfqpoint{0cm}{1.94cm}}
\pgfpathclose
\pgfusepath{clip}
\begin{pgfscope}
\begin{pgfscope}
\pgfsetdash{}{0cm}
\pgfsetlinewidth{0.818mm}
\pgfsetroundcap
\pgfsetroundjoin
\pgfsetmiterlimit{7.0}
\definecolor{eps2pgf_color}{gray}{0}\pgfsetstrokecolor{eps2pgf_color}\pgfsetfillcolor{eps2pgf_color}
\pgfpathmoveto{\pgfqpoint{0.117cm}{1.815cm}}
\pgfpathlineto{\pgfqpoint{0.682cm}{1.065cm}}
\pgfpathlineto{\pgfqpoint{1.246cm}{1.815cm}}
\pgfusepath{stroke}
\end{pgfscope}
\definecolor{eps2pgf_color}{gray}{0}\pgfsetstrokecolor{eps2pgf_color}\pgfsetfillcolor{eps2pgf_color}
\pgfpathmoveto{\pgfqpoint{0.273cm}{1.789cm}}
\pgfpathcurveto{\pgfqpoint{0.273cm}{1.825cm}}{\pgfqpoint{0.259cm}{1.86cm}}{\pgfqpoint{0.233cm}{1.886cm}}
\pgfpathcurveto{\pgfqpoint{0.207cm}{1.912cm}}{\pgfqpoint{0.173cm}{1.926cm}}{\pgfqpoint{0.137cm}{1.926cm}}
\pgfpathcurveto{\pgfqpoint{0.1cm}{1.926cm}}{\pgfqpoint{0.066cm}{1.912cm}}{\pgfqpoint{0.04cm}{1.886cm}}
\pgfpathcurveto{\pgfqpoint{0.014cm}{1.86cm}}{\pgfqpoint{0cm}{1.825cm}}{\pgfqpoint{0cm}{1.789cm}}
\pgfpathcurveto{\pgfqpoint{0cm}{1.753cm}}{\pgfqpoint{0.014cm}{1.718cm}}{\pgfqpoint{0.04cm}{1.692cm}}
\pgfpathcurveto{\pgfqpoint{0.066cm}{1.667cm}}{\pgfqpoint{0.1cm}{1.652cm}}{\pgfqpoint{0.137cm}{1.652cm}}
\pgfpathcurveto{\pgfqpoint{0.173cm}{1.652cm}}{\pgfqpoint{0.207cm}{1.667cm}}{\pgfqpoint{0.233cm}{1.692cm}}
\pgfpathcurveto{\pgfqpoint{0.259cm}{1.718cm}}{\pgfqpoint{0.273cm}{1.753cm}}{\pgfqpoint{0.273cm}{1.789cm}}
\pgfusepath{fill}
\pgfpathmoveto{\pgfqpoint{1.345cm}{1.765cm}}
\pgfpathcurveto{\pgfqpoint{1.345cm}{1.801cm}}{\pgfqpoint{1.331cm}{1.836cm}}{\pgfqpoint{1.305cm}{1.862cm}}
\pgfpathcurveto{\pgfqpoint{1.28cm}{1.887cm}}{\pgfqpoint{1.245cm}{1.902cm}}{\pgfqpoint{1.209cm}{1.902cm}}
\pgfpathcurveto{\pgfqpoint{1.172cm}{1.902cm}}{\pgfqpoint{1.138cm}{1.887cm}}{\pgfqpoint{1.112cm}{1.862cm}}
\pgfpathcurveto{\pgfqpoint{1.087cm}{1.836cm}}{\pgfqpoint{1.072cm}{1.801cm}}{\pgfqpoint{1.072cm}{1.765cm}}
\pgfpathcurveto{\pgfqpoint{1.072cm}{1.728cm}}{\pgfqpoint{1.087cm}{1.694cm}}{\pgfqpoint{1.112cm}{1.668cm}}
\pgfpathcurveto{\pgfqpoint{1.138cm}{1.642cm}}{\pgfqpoint{1.172cm}{1.628cm}}{\pgfqpoint{1.209cm}{1.628cm}}
\pgfpathcurveto{\pgfqpoint{1.245cm}{1.628cm}}{\pgfqpoint{1.28cm}{1.642cm}}{\pgfqpoint{1.305cm}{1.668cm}}
\pgfpathcurveto{\pgfqpoint{1.331cm}{1.694cm}}{\pgfqpoint{1.345cm}{1.728cm}}{\pgfqpoint{1.345cm}{1.765cm}}
\pgfusepath{fill}
\begin{pgfscope}
\pgfsetdash{}{0cm}
\pgfsetlinewidth{0.818mm}
\pgfsetroundcap
\pgfsetroundjoin
\pgfsetmiterlimit{7.0}
\pgfpathmoveto{\pgfqpoint{0.682cm}{1.065cm}}
\pgfpathlineto{\pgfqpoint{1.246cm}{0.315cm}}
\pgfpathlineto{\pgfqpoint{1.811cm}{1.065cm}}
\pgfusepath{stroke}
\end{pgfscope}
\pgfpathmoveto{\pgfqpoint{1.948cm}{1.065cm}}
\pgfpathcurveto{\pgfqpoint{1.948cm}{1.101cm}}{\pgfqpoint{1.933cm}{1.136cm}}{\pgfqpoint{1.907cm}{1.162cm}}
\pgfpathcurveto{\pgfqpoint{1.882cm}{1.187cm}}{\pgfqpoint{1.847cm}{1.202cm}}{\pgfqpoint{1.811cm}{1.202cm}}
\pgfpathcurveto{\pgfqpoint{1.775cm}{1.202cm}}{\pgfqpoint{1.74cm}{1.187cm}}{\pgfqpoint{1.714cm}{1.162cm}}
\pgfpathcurveto{\pgfqpoint{1.689cm}{1.136cm}}{\pgfqpoint{1.674cm}{1.101cm}}{\pgfqpoint{1.674cm}{1.065cm}}
\pgfpathcurveto{\pgfqpoint{1.674cm}{1.029cm}}{\pgfqpoint{1.689cm}{0.994cm}}{\pgfqpoint{1.714cm}{0.968cm}}
\pgfpathcurveto{\pgfqpoint{1.74cm}{0.942cm}}{\pgfqpoint{1.775cm}{0.928cm}}{\pgfqpoint{1.811cm}{0.928cm}}
\pgfpathcurveto{\pgfqpoint{1.847cm}{0.928cm}}{\pgfqpoint{1.882cm}{0.942cm}}{\pgfqpoint{1.907cm}{0.968cm}}
\pgfpathcurveto{\pgfqpoint{1.933cm}{0.994cm}}{\pgfqpoint{1.948cm}{1.029cm}}{\pgfqpoint{1.948cm}{1.065cm}}
\pgfusepath{fill}
\begin{pgfscope}
\pgfsetdash{}{0cm}
\pgfsetlinewidth{0.818mm}
\pgfsetmiterlimit{7.0}
\pgfpathmoveto{\pgfqpoint{1.246cm}{0.315cm}}
\pgfpathlineto{\pgfqpoint{1.244cm}{1.061cm}}
\pgfusepath{stroke}
\end{pgfscope}
\pgfpathmoveto{\pgfqpoint{1.38cm}{1.065cm}}
\pgfpathcurveto{\pgfqpoint{1.38cm}{1.101cm}}{\pgfqpoint{1.366cm}{1.136cm}}{\pgfqpoint{1.34cm}{1.162cm}}
\pgfpathcurveto{\pgfqpoint{1.315cm}{1.187cm}}{\pgfqpoint{1.28cm}{1.202cm}}{\pgfqpoint{1.244cm}{1.202cm}}
\pgfpathcurveto{\pgfqpoint{1.207cm}{1.202cm}}{\pgfqpoint{1.173cm}{1.187cm}}{\pgfqpoint{1.147cm}{1.162cm}}
\pgfpathcurveto{\pgfqpoint{1.121cm}{1.136cm}}{\pgfqpoint{1.107cm}{1.101cm}}{\pgfqpoint{1.107cm}{1.065cm}}
\pgfpathcurveto{\pgfqpoint{1.107cm}{1.029cm}}{\pgfqpoint{1.121cm}{0.994cm}}{\pgfqpoint{1.147cm}{0.968cm}}
\pgfpathcurveto{\pgfqpoint{1.173cm}{0.942cm}}{\pgfqpoint{1.207cm}{0.928cm}}{\pgfqpoint{1.244cm}{0.928cm}}
\pgfpathcurveto{\pgfqpoint{1.28cm}{0.928cm}}{\pgfqpoint{1.315cm}{0.942cm}}{\pgfqpoint{1.34cm}{0.968cm}}
\pgfpathcurveto{\pgfqpoint{1.366cm}{0.994cm}}{\pgfqpoint{1.38cm}{1.029cm}}{\pgfqpoint{1.38cm}{1.065cm}}
\pgfusepath{fill}
\begin{pgfscope}
\pgfsetdash{}{0cm}
\pgfsetlinewidth{0.818mm}
\pgfsetmiterlimit{4.0}
\pgfpathmoveto{\pgfqpoint{1.383cm}{0.178cm}}
\pgfpathcurveto{\pgfqpoint{1.383cm}{0.214cm}}{\pgfqpoint{1.369cm}{0.249cm}}{\pgfqpoint{1.343cm}{0.275cm}}
\pgfpathcurveto{\pgfqpoint{1.317cm}{0.3cm}}{\pgfqpoint{1.283cm}{0.315cm}}{\pgfqpoint{1.246cm}{0.315cm}}
\pgfpathcurveto{\pgfqpoint{1.21cm}{0.315cm}}{\pgfqpoint{1.175cm}{0.3cm}}{\pgfqpoint{1.15cm}{0.275cm}}
\pgfpathcurveto{\pgfqpoint{1.124cm}{0.249cm}}{\pgfqpoint{1.11cm}{0.214cm}}{\pgfqpoint{1.11cm}{0.178cm}}
\pgfpathcurveto{\pgfqpoint{1.11cm}{0.141cm}}{\pgfqpoint{1.124cm}{0.107cm}}{\pgfqpoint{1.15cm}{0.081cm}}
\pgfpathcurveto{\pgfqpoint{1.175cm}{0.055cm}}{\pgfqpoint{1.21cm}{0.041cm}}{\pgfqpoint{1.246cm}{0.041cm}}
\pgfpathcurveto{\pgfqpoint{1.283cm}{0.041cm}}{\pgfqpoint{1.317cm}{0.055cm}}{\pgfqpoint{1.343cm}{0.081cm}}
\pgfpathcurveto{\pgfqpoint{1.369cm}{0.107cm}}{\pgfqpoint{1.383cm}{0.141cm}}{\pgfqpoint{1.383cm}{0.178cm}}
\pgfusepath{stroke}
\end{pgfscope}
\end{pgfscope}
\end{pgfscope}
\end{pgfscope}
\end{tikzpicture}}}\|_{\CC^{-2+\alpha}(\rho)}
\lesssim  \|\phi+\psi\|_{L^\infty(\rho)}\|\UU_>X^{\!\resizebox{!}{.8em}{
\begin{tikzpicture}
\pgfpathmoveto{\pgfqpoint{0cm}{-0.035cm}}
\pgfpathlineto{\pgfqpoint{1.976cm}{-0.035cm}}
\pgfpathlineto{\pgfqpoint{1.976cm}{1.94cm}}
\pgfpathlineto{\pgfqpoint{0cm}{1.94cm}}
\pgfpathclose
\pgfusepath{clip}
\begin{pgfscope}
\begin{pgfscope}
\pgfpathmoveto{\pgfqpoint{0cm}{-0.035cm}}
\pgfpathlineto{\pgfqpoint{1.976cm}{-0.035cm}}
\pgfpathlineto{\pgfqpoint{1.976cm}{1.94cm}}
\pgfpathlineto{\pgfqpoint{0cm}{1.94cm}}
\pgfpathclose
\pgfusepath{clip}
\begin{pgfscope}
\begin{pgfscope}
\pgfsetdash{}{0cm}
\pgfsetlinewidth{0.818mm}
\pgfsetroundcap
\pgfsetroundjoin
\pgfsetmiterlimit{7.0}
\definecolor{eps2pgf_color}{gray}{0}\pgfsetstrokecolor{eps2pgf_color}\pgfsetfillcolor{eps2pgf_color}
\pgfpathmoveto{\pgfqpoint{0.117cm}{1.815cm}}
\pgfpathlineto{\pgfqpoint{0.682cm}{1.065cm}}
\pgfpathlineto{\pgfqpoint{1.246cm}{1.815cm}}
\pgfusepath{stroke}
\end{pgfscope}
\definecolor{eps2pgf_color}{gray}{0}\pgfsetstrokecolor{eps2pgf_color}\pgfsetfillcolor{eps2pgf_color}
\pgfpathmoveto{\pgfqpoint{0.273cm}{1.789cm}}
\pgfpathcurveto{\pgfqpoint{0.273cm}{1.825cm}}{\pgfqpoint{0.259cm}{1.86cm}}{\pgfqpoint{0.233cm}{1.886cm}}
\pgfpathcurveto{\pgfqpoint{0.207cm}{1.912cm}}{\pgfqpoint{0.173cm}{1.926cm}}{\pgfqpoint{0.137cm}{1.926cm}}
\pgfpathcurveto{\pgfqpoint{0.1cm}{1.926cm}}{\pgfqpoint{0.066cm}{1.912cm}}{\pgfqpoint{0.04cm}{1.886cm}}
\pgfpathcurveto{\pgfqpoint{0.014cm}{1.86cm}}{\pgfqpoint{0cm}{1.825cm}}{\pgfqpoint{0cm}{1.789cm}}
\pgfpathcurveto{\pgfqpoint{0cm}{1.753cm}}{\pgfqpoint{0.014cm}{1.718cm}}{\pgfqpoint{0.04cm}{1.692cm}}
\pgfpathcurveto{\pgfqpoint{0.066cm}{1.667cm}}{\pgfqpoint{0.1cm}{1.652cm}}{\pgfqpoint{0.137cm}{1.652cm}}
\pgfpathcurveto{\pgfqpoint{0.173cm}{1.652cm}}{\pgfqpoint{0.207cm}{1.667cm}}{\pgfqpoint{0.233cm}{1.692cm}}
\pgfpathcurveto{\pgfqpoint{0.259cm}{1.718cm}}{\pgfqpoint{0.273cm}{1.753cm}}{\pgfqpoint{0.273cm}{1.789cm}}
\pgfusepath{fill}
\pgfpathmoveto{\pgfqpoint{1.345cm}{1.765cm}}
\pgfpathcurveto{\pgfqpoint{1.345cm}{1.801cm}}{\pgfqpoint{1.331cm}{1.836cm}}{\pgfqpoint{1.305cm}{1.862cm}}
\pgfpathcurveto{\pgfqpoint{1.28cm}{1.887cm}}{\pgfqpoint{1.245cm}{1.902cm}}{\pgfqpoint{1.209cm}{1.902cm}}
\pgfpathcurveto{\pgfqpoint{1.172cm}{1.902cm}}{\pgfqpoint{1.138cm}{1.887cm}}{\pgfqpoint{1.112cm}{1.862cm}}
\pgfpathcurveto{\pgfqpoint{1.087cm}{1.836cm}}{\pgfqpoint{1.072cm}{1.801cm}}{\pgfqpoint{1.072cm}{1.765cm}}
\pgfpathcurveto{\pgfqpoint{1.072cm}{1.728cm}}{\pgfqpoint{1.087cm}{1.694cm}}{\pgfqpoint{1.112cm}{1.668cm}}
\pgfpathcurveto{\pgfqpoint{1.138cm}{1.642cm}}{\pgfqpoint{1.172cm}{1.628cm}}{\pgfqpoint{1.209cm}{1.628cm}}
\pgfpathcurveto{\pgfqpoint{1.245cm}{1.628cm}}{\pgfqpoint{1.28cm}{1.642cm}}{\pgfqpoint{1.305cm}{1.668cm}}
\pgfpathcurveto{\pgfqpoint{1.331cm}{1.694cm}}{\pgfqpoint{1.345cm}{1.728cm}}{\pgfqpoint{1.345cm}{1.765cm}}
\pgfusepath{fill}
\begin{pgfscope}
\pgfsetdash{}{0cm}
\pgfsetlinewidth{0.818mm}
\pgfsetroundcap
\pgfsetroundjoin
\pgfsetmiterlimit{7.0}
\pgfpathmoveto{\pgfqpoint{0.682cm}{1.065cm}}
\pgfpathlineto{\pgfqpoint{1.246cm}{0.315cm}}
\pgfpathlineto{\pgfqpoint{1.811cm}{1.065cm}}
\pgfusepath{stroke}
\end{pgfscope}
\pgfpathmoveto{\pgfqpoint{1.948cm}{1.065cm}}
\pgfpathcurveto{\pgfqpoint{1.948cm}{1.101cm}}{\pgfqpoint{1.933cm}{1.136cm}}{\pgfqpoint{1.907cm}{1.162cm}}
\pgfpathcurveto{\pgfqpoint{1.882cm}{1.187cm}}{\pgfqpoint{1.847cm}{1.202cm}}{\pgfqpoint{1.811cm}{1.202cm}}
\pgfpathcurveto{\pgfqpoint{1.775cm}{1.202cm}}{\pgfqpoint{1.74cm}{1.187cm}}{\pgfqpoint{1.714cm}{1.162cm}}
\pgfpathcurveto{\pgfqpoint{1.689cm}{1.136cm}}{\pgfqpoint{1.674cm}{1.101cm}}{\pgfqpoint{1.674cm}{1.065cm}}
\pgfpathcurveto{\pgfqpoint{1.674cm}{1.029cm}}{\pgfqpoint{1.689cm}{0.994cm}}{\pgfqpoint{1.714cm}{0.968cm}}
\pgfpathcurveto{\pgfqpoint{1.74cm}{0.942cm}}{\pgfqpoint{1.775cm}{0.928cm}}{\pgfqpoint{1.811cm}{0.928cm}}
\pgfpathcurveto{\pgfqpoint{1.847cm}{0.928cm}}{\pgfqpoint{1.882cm}{0.942cm}}{\pgfqpoint{1.907cm}{0.968cm}}
\pgfpathcurveto{\pgfqpoint{1.933cm}{0.994cm}}{\pgfqpoint{1.948cm}{1.029cm}}{\pgfqpoint{1.948cm}{1.065cm}}
\pgfusepath{fill}
\begin{pgfscope}
\pgfsetdash{}{0cm}
\pgfsetlinewidth{0.818mm}
\pgfsetmiterlimit{7.0}
\pgfpathmoveto{\pgfqpoint{1.246cm}{0.315cm}}
\pgfpathlineto{\pgfqpoint{1.244cm}{1.061cm}}
\pgfusepath{stroke}
\end{pgfscope}
\pgfpathmoveto{\pgfqpoint{1.38cm}{1.065cm}}
\pgfpathcurveto{\pgfqpoint{1.38cm}{1.101cm}}{\pgfqpoint{1.366cm}{1.136cm}}{\pgfqpoint{1.34cm}{1.162cm}}
\pgfpathcurveto{\pgfqpoint{1.315cm}{1.187cm}}{\pgfqpoint{1.28cm}{1.202cm}}{\pgfqpoint{1.244cm}{1.202cm}}
\pgfpathcurveto{\pgfqpoint{1.207cm}{1.202cm}}{\pgfqpoint{1.173cm}{1.187cm}}{\pgfqpoint{1.147cm}{1.162cm}}
\pgfpathcurveto{\pgfqpoint{1.121cm}{1.136cm}}{\pgfqpoint{1.107cm}{1.101cm}}{\pgfqpoint{1.107cm}{1.065cm}}
\pgfpathcurveto{\pgfqpoint{1.107cm}{1.029cm}}{\pgfqpoint{1.121cm}{0.994cm}}{\pgfqpoint{1.147cm}{0.968cm}}
\pgfpathcurveto{\pgfqpoint{1.173cm}{0.942cm}}{\pgfqpoint{1.207cm}{0.928cm}}{\pgfqpoint{1.244cm}{0.928cm}}
\pgfpathcurveto{\pgfqpoint{1.28cm}{0.928cm}}{\pgfqpoint{1.315cm}{0.942cm}}{\pgfqpoint{1.34cm}{0.968cm}}
\pgfpathcurveto{\pgfqpoint{1.366cm}{0.994cm}}{\pgfqpoint{1.38cm}{1.029cm}}{\pgfqpoint{1.38cm}{1.065cm}}
\pgfusepath{fill}
\begin{pgfscope}
\pgfsetdash{}{0cm}
\pgfsetlinewidth{0.818mm}
\pgfsetmiterlimit{4.0}
\pgfpathmoveto{\pgfqpoint{1.383cm}{0.178cm}}
\pgfpathcurveto{\pgfqpoint{1.383cm}{0.214cm}}{\pgfqpoint{1.369cm}{0.249cm}}{\pgfqpoint{1.343cm}{0.275cm}}
\pgfpathcurveto{\pgfqpoint{1.317cm}{0.3cm}}{\pgfqpoint{1.283cm}{0.315cm}}{\pgfqpoint{1.246cm}{0.315cm}}
\pgfpathcurveto{\pgfqpoint{1.21cm}{0.315cm}}{\pgfqpoint{1.175cm}{0.3cm}}{\pgfqpoint{1.15cm}{0.275cm}}
\pgfpathcurveto{\pgfqpoint{1.124cm}{0.249cm}}{\pgfqpoint{1.11cm}{0.214cm}}{\pgfqpoint{1.11cm}{0.178cm}}
\pgfpathcurveto{\pgfqpoint{1.11cm}{0.141cm}}{\pgfqpoint{1.124cm}{0.107cm}}{\pgfqpoint{1.15cm}{0.081cm}}
\pgfpathcurveto{\pgfqpoint{1.175cm}{0.055cm}}{\pgfqpoint{1.21cm}{0.041cm}}{\pgfqpoint{1.246cm}{0.041cm}}
\pgfpathcurveto{\pgfqpoint{1.283cm}{0.041cm}}{\pgfqpoint{1.317cm}{0.055cm}}{\pgfqpoint{1.343cm}{0.081cm}}
\pgfpathcurveto{\pgfqpoint{1.369cm}{0.107cm}}{\pgfqpoint{1.383cm}{0.141cm}}{\pgfqpoint{1.383cm}{0.178cm}}
\pgfusepath{stroke}
\end{pgfscope}
\end{pgfscope}
\end{pgfscope}
\end{pgfscope}
\end{tikzpicture}}}\|_{\CC^{-2+\alpha}}\lesssim 2^{-(2-\alpha-\kappa)K/2}\|\phi+\psi\|_{L^\infty(\rho)},
\end{align*}
provided
\begin{equation}\label{eq:l2}
\|\UU_> X^{\!\resizebox{!}{.8em}{
\begin{tikzpicture}
\pgfpathmoveto{\pgfqpoint{0cm}{-0.035cm}}
\pgfpathlineto{\pgfqpoint{1.976cm}{-0.035cm}}
\pgfpathlineto{\pgfqpoint{1.976cm}{1.94cm}}
\pgfpathlineto{\pgfqpoint{0cm}{1.94cm}}
\pgfpathclose
\pgfusepath{clip}
\begin{pgfscope}
\begin{pgfscope}
\pgfpathmoveto{\pgfqpoint{0cm}{-0.035cm}}
\pgfpathlineto{\pgfqpoint{1.976cm}{-0.035cm}}
\pgfpathlineto{\pgfqpoint{1.976cm}{1.94cm}}
\pgfpathlineto{\pgfqpoint{0cm}{1.94cm}}
\pgfpathclose
\pgfusepath{clip}
\begin{pgfscope}
\begin{pgfscope}
\pgfsetdash{}{0cm}
\pgfsetlinewidth{0.818mm}
\pgfsetroundcap
\pgfsetroundjoin
\pgfsetmiterlimit{7.0}
\definecolor{eps2pgf_color}{gray}{0}\pgfsetstrokecolor{eps2pgf_color}\pgfsetfillcolor{eps2pgf_color}
\pgfpathmoveto{\pgfqpoint{0.117cm}{1.815cm}}
\pgfpathlineto{\pgfqpoint{0.682cm}{1.065cm}}
\pgfpathlineto{\pgfqpoint{1.246cm}{1.815cm}}
\pgfusepath{stroke}
\end{pgfscope}
\definecolor{eps2pgf_color}{gray}{0}\pgfsetstrokecolor{eps2pgf_color}\pgfsetfillcolor{eps2pgf_color}
\pgfpathmoveto{\pgfqpoint{0.273cm}{1.789cm}}
\pgfpathcurveto{\pgfqpoint{0.273cm}{1.825cm}}{\pgfqpoint{0.259cm}{1.86cm}}{\pgfqpoint{0.233cm}{1.886cm}}
\pgfpathcurveto{\pgfqpoint{0.207cm}{1.912cm}}{\pgfqpoint{0.173cm}{1.926cm}}{\pgfqpoint{0.137cm}{1.926cm}}
\pgfpathcurveto{\pgfqpoint{0.1cm}{1.926cm}}{\pgfqpoint{0.066cm}{1.912cm}}{\pgfqpoint{0.04cm}{1.886cm}}
\pgfpathcurveto{\pgfqpoint{0.014cm}{1.86cm}}{\pgfqpoint{0cm}{1.825cm}}{\pgfqpoint{0cm}{1.789cm}}
\pgfpathcurveto{\pgfqpoint{0cm}{1.753cm}}{\pgfqpoint{0.014cm}{1.718cm}}{\pgfqpoint{0.04cm}{1.692cm}}
\pgfpathcurveto{\pgfqpoint{0.066cm}{1.667cm}}{\pgfqpoint{0.1cm}{1.652cm}}{\pgfqpoint{0.137cm}{1.652cm}}
\pgfpathcurveto{\pgfqpoint{0.173cm}{1.652cm}}{\pgfqpoint{0.207cm}{1.667cm}}{\pgfqpoint{0.233cm}{1.692cm}}
\pgfpathcurveto{\pgfqpoint{0.259cm}{1.718cm}}{\pgfqpoint{0.273cm}{1.753cm}}{\pgfqpoint{0.273cm}{1.789cm}}
\pgfusepath{fill}
\pgfpathmoveto{\pgfqpoint{1.345cm}{1.765cm}}
\pgfpathcurveto{\pgfqpoint{1.345cm}{1.801cm}}{\pgfqpoint{1.331cm}{1.836cm}}{\pgfqpoint{1.305cm}{1.862cm}}
\pgfpathcurveto{\pgfqpoint{1.28cm}{1.887cm}}{\pgfqpoint{1.245cm}{1.902cm}}{\pgfqpoint{1.209cm}{1.902cm}}
\pgfpathcurveto{\pgfqpoint{1.172cm}{1.902cm}}{\pgfqpoint{1.138cm}{1.887cm}}{\pgfqpoint{1.112cm}{1.862cm}}
\pgfpathcurveto{\pgfqpoint{1.087cm}{1.836cm}}{\pgfqpoint{1.072cm}{1.801cm}}{\pgfqpoint{1.072cm}{1.765cm}}
\pgfpathcurveto{\pgfqpoint{1.072cm}{1.728cm}}{\pgfqpoint{1.087cm}{1.694cm}}{\pgfqpoint{1.112cm}{1.668cm}}
\pgfpathcurveto{\pgfqpoint{1.138cm}{1.642cm}}{\pgfqpoint{1.172cm}{1.628cm}}{\pgfqpoint{1.209cm}{1.628cm}}
\pgfpathcurveto{\pgfqpoint{1.245cm}{1.628cm}}{\pgfqpoint{1.28cm}{1.642cm}}{\pgfqpoint{1.305cm}{1.668cm}}
\pgfpathcurveto{\pgfqpoint{1.331cm}{1.694cm}}{\pgfqpoint{1.345cm}{1.728cm}}{\pgfqpoint{1.345cm}{1.765cm}}
\pgfusepath{fill}
\begin{pgfscope}
\pgfsetdash{}{0cm}
\pgfsetlinewidth{0.818mm}
\pgfsetroundcap
\pgfsetroundjoin
\pgfsetmiterlimit{7.0}
\pgfpathmoveto{\pgfqpoint{0.682cm}{1.065cm}}
\pgfpathlineto{\pgfqpoint{1.246cm}{0.315cm}}
\pgfpathlineto{\pgfqpoint{1.811cm}{1.065cm}}
\pgfusepath{stroke}
\end{pgfscope}
\pgfpathmoveto{\pgfqpoint{1.948cm}{1.065cm}}
\pgfpathcurveto{\pgfqpoint{1.948cm}{1.101cm}}{\pgfqpoint{1.933cm}{1.136cm}}{\pgfqpoint{1.907cm}{1.162cm}}
\pgfpathcurveto{\pgfqpoint{1.882cm}{1.187cm}}{\pgfqpoint{1.847cm}{1.202cm}}{\pgfqpoint{1.811cm}{1.202cm}}
\pgfpathcurveto{\pgfqpoint{1.775cm}{1.202cm}}{\pgfqpoint{1.74cm}{1.187cm}}{\pgfqpoint{1.714cm}{1.162cm}}
\pgfpathcurveto{\pgfqpoint{1.689cm}{1.136cm}}{\pgfqpoint{1.674cm}{1.101cm}}{\pgfqpoint{1.674cm}{1.065cm}}
\pgfpathcurveto{\pgfqpoint{1.674cm}{1.029cm}}{\pgfqpoint{1.689cm}{0.994cm}}{\pgfqpoint{1.714cm}{0.968cm}}
\pgfpathcurveto{\pgfqpoint{1.74cm}{0.942cm}}{\pgfqpoint{1.775cm}{0.928cm}}{\pgfqpoint{1.811cm}{0.928cm}}
\pgfpathcurveto{\pgfqpoint{1.847cm}{0.928cm}}{\pgfqpoint{1.882cm}{0.942cm}}{\pgfqpoint{1.907cm}{0.968cm}}
\pgfpathcurveto{\pgfqpoint{1.933cm}{0.994cm}}{\pgfqpoint{1.948cm}{1.029cm}}{\pgfqpoint{1.948cm}{1.065cm}}
\pgfusepath{fill}
\begin{pgfscope}
\pgfsetdash{}{0cm}
\pgfsetlinewidth{0.818mm}
\pgfsetmiterlimit{7.0}
\pgfpathmoveto{\pgfqpoint{1.246cm}{0.315cm}}
\pgfpathlineto{\pgfqpoint{1.244cm}{1.061cm}}
\pgfusepath{stroke}
\end{pgfscope}
\pgfpathmoveto{\pgfqpoint{1.38cm}{1.065cm}}
\pgfpathcurveto{\pgfqpoint{1.38cm}{1.101cm}}{\pgfqpoint{1.366cm}{1.136cm}}{\pgfqpoint{1.34cm}{1.162cm}}
\pgfpathcurveto{\pgfqpoint{1.315cm}{1.187cm}}{\pgfqpoint{1.28cm}{1.202cm}}{\pgfqpoint{1.244cm}{1.202cm}}
\pgfpathcurveto{\pgfqpoint{1.207cm}{1.202cm}}{\pgfqpoint{1.173cm}{1.187cm}}{\pgfqpoint{1.147cm}{1.162cm}}
\pgfpathcurveto{\pgfqpoint{1.121cm}{1.136cm}}{\pgfqpoint{1.107cm}{1.101cm}}{\pgfqpoint{1.107cm}{1.065cm}}
\pgfpathcurveto{\pgfqpoint{1.107cm}{1.029cm}}{\pgfqpoint{1.121cm}{0.994cm}}{\pgfqpoint{1.147cm}{0.968cm}}
\pgfpathcurveto{\pgfqpoint{1.173cm}{0.942cm}}{\pgfqpoint{1.207cm}{0.928cm}}{\pgfqpoint{1.244cm}{0.928cm}}
\pgfpathcurveto{\pgfqpoint{1.28cm}{0.928cm}}{\pgfqpoint{1.315cm}{0.942cm}}{\pgfqpoint{1.34cm}{0.968cm}}
\pgfpathcurveto{\pgfqpoint{1.366cm}{0.994cm}}{\pgfqpoint{1.38cm}{1.029cm}}{\pgfqpoint{1.38cm}{1.065cm}}
\pgfusepath{fill}
\begin{pgfscope}
\pgfsetdash{}{0cm}
\pgfsetlinewidth{0.818mm}
\pgfsetmiterlimit{4.0}
\pgfpathmoveto{\pgfqpoint{1.383cm}{0.178cm}}
\pgfpathcurveto{\pgfqpoint{1.383cm}{0.214cm}}{\pgfqpoint{1.369cm}{0.249cm}}{\pgfqpoint{1.343cm}{0.275cm}}
\pgfpathcurveto{\pgfqpoint{1.317cm}{0.3cm}}{\pgfqpoint{1.283cm}{0.315cm}}{\pgfqpoint{1.246cm}{0.315cm}}
\pgfpathcurveto{\pgfqpoint{1.21cm}{0.315cm}}{\pgfqpoint{1.175cm}{0.3cm}}{\pgfqpoint{1.15cm}{0.275cm}}
\pgfpathcurveto{\pgfqpoint{1.124cm}{0.249cm}}{\pgfqpoint{1.11cm}{0.214cm}}{\pgfqpoint{1.11cm}{0.178cm}}
\pgfpathcurveto{\pgfqpoint{1.11cm}{0.141cm}}{\pgfqpoint{1.124cm}{0.107cm}}{\pgfqpoint{1.15cm}{0.081cm}}
\pgfpathcurveto{\pgfqpoint{1.175cm}{0.055cm}}{\pgfqpoint{1.21cm}{0.041cm}}{\pgfqpoint{1.246cm}{0.041cm}}
\pgfpathcurveto{\pgfqpoint{1.283cm}{0.041cm}}{\pgfqpoint{1.317cm}{0.055cm}}{\pgfqpoint{1.343cm}{0.081cm}}
\pgfpathcurveto{\pgfqpoint{1.369cm}{0.107cm}}{\pgfqpoint{1.383cm}{0.141cm}}{\pgfqpoint{1.383cm}{0.178cm}}
\pgfusepath{stroke}
\end{pgfscope}
\end{pgfscope}
\end{pgfscope}
\end{pgfscope}
\end{tikzpicture}}}\|_{\CC^{-2+\alpha}}\lesssim 2^{-(2-\alpha-\kappa)K/2} \|X^{\!\resizebox{!}{.8em}{
\begin{tikzpicture}
\pgfpathmoveto{\pgfqpoint{0cm}{-0.035cm}}
\pgfpathlineto{\pgfqpoint{1.976cm}{-0.035cm}}
\pgfpathlineto{\pgfqpoint{1.976cm}{1.94cm}}
\pgfpathlineto{\pgfqpoint{0cm}{1.94cm}}
\pgfpathclose
\pgfusepath{clip}
\begin{pgfscope}
\begin{pgfscope}
\pgfpathmoveto{\pgfqpoint{0cm}{-0.035cm}}
\pgfpathlineto{\pgfqpoint{1.976cm}{-0.035cm}}
\pgfpathlineto{\pgfqpoint{1.976cm}{1.94cm}}
\pgfpathlineto{\pgfqpoint{0cm}{1.94cm}}
\pgfpathclose
\pgfusepath{clip}
\begin{pgfscope}
\begin{pgfscope}
\pgfsetdash{}{0cm}
\pgfsetlinewidth{0.818mm}
\pgfsetroundcap
\pgfsetroundjoin
\pgfsetmiterlimit{7.0}
\definecolor{eps2pgf_color}{gray}{0}\pgfsetstrokecolor{eps2pgf_color}\pgfsetfillcolor{eps2pgf_color}
\pgfpathmoveto{\pgfqpoint{0.117cm}{1.815cm}}
\pgfpathlineto{\pgfqpoint{0.682cm}{1.065cm}}
\pgfpathlineto{\pgfqpoint{1.246cm}{1.815cm}}
\pgfusepath{stroke}
\end{pgfscope}
\definecolor{eps2pgf_color}{gray}{0}\pgfsetstrokecolor{eps2pgf_color}\pgfsetfillcolor{eps2pgf_color}
\pgfpathmoveto{\pgfqpoint{0.273cm}{1.789cm}}
\pgfpathcurveto{\pgfqpoint{0.273cm}{1.825cm}}{\pgfqpoint{0.259cm}{1.86cm}}{\pgfqpoint{0.233cm}{1.886cm}}
\pgfpathcurveto{\pgfqpoint{0.207cm}{1.912cm}}{\pgfqpoint{0.173cm}{1.926cm}}{\pgfqpoint{0.137cm}{1.926cm}}
\pgfpathcurveto{\pgfqpoint{0.1cm}{1.926cm}}{\pgfqpoint{0.066cm}{1.912cm}}{\pgfqpoint{0.04cm}{1.886cm}}
\pgfpathcurveto{\pgfqpoint{0.014cm}{1.86cm}}{\pgfqpoint{0cm}{1.825cm}}{\pgfqpoint{0cm}{1.789cm}}
\pgfpathcurveto{\pgfqpoint{0cm}{1.753cm}}{\pgfqpoint{0.014cm}{1.718cm}}{\pgfqpoint{0.04cm}{1.692cm}}
\pgfpathcurveto{\pgfqpoint{0.066cm}{1.667cm}}{\pgfqpoint{0.1cm}{1.652cm}}{\pgfqpoint{0.137cm}{1.652cm}}
\pgfpathcurveto{\pgfqpoint{0.173cm}{1.652cm}}{\pgfqpoint{0.207cm}{1.667cm}}{\pgfqpoint{0.233cm}{1.692cm}}
\pgfpathcurveto{\pgfqpoint{0.259cm}{1.718cm}}{\pgfqpoint{0.273cm}{1.753cm}}{\pgfqpoint{0.273cm}{1.789cm}}
\pgfusepath{fill}
\pgfpathmoveto{\pgfqpoint{1.345cm}{1.765cm}}
\pgfpathcurveto{\pgfqpoint{1.345cm}{1.801cm}}{\pgfqpoint{1.331cm}{1.836cm}}{\pgfqpoint{1.305cm}{1.862cm}}
\pgfpathcurveto{\pgfqpoint{1.28cm}{1.887cm}}{\pgfqpoint{1.245cm}{1.902cm}}{\pgfqpoint{1.209cm}{1.902cm}}
\pgfpathcurveto{\pgfqpoint{1.172cm}{1.902cm}}{\pgfqpoint{1.138cm}{1.887cm}}{\pgfqpoint{1.112cm}{1.862cm}}
\pgfpathcurveto{\pgfqpoint{1.087cm}{1.836cm}}{\pgfqpoint{1.072cm}{1.801cm}}{\pgfqpoint{1.072cm}{1.765cm}}
\pgfpathcurveto{\pgfqpoint{1.072cm}{1.728cm}}{\pgfqpoint{1.087cm}{1.694cm}}{\pgfqpoint{1.112cm}{1.668cm}}
\pgfpathcurveto{\pgfqpoint{1.138cm}{1.642cm}}{\pgfqpoint{1.172cm}{1.628cm}}{\pgfqpoint{1.209cm}{1.628cm}}
\pgfpathcurveto{\pgfqpoint{1.245cm}{1.628cm}}{\pgfqpoint{1.28cm}{1.642cm}}{\pgfqpoint{1.305cm}{1.668cm}}
\pgfpathcurveto{\pgfqpoint{1.331cm}{1.694cm}}{\pgfqpoint{1.345cm}{1.728cm}}{\pgfqpoint{1.345cm}{1.765cm}}
\pgfusepath{fill}
\begin{pgfscope}
\pgfsetdash{}{0cm}
\pgfsetlinewidth{0.818mm}
\pgfsetroundcap
\pgfsetroundjoin
\pgfsetmiterlimit{7.0}
\pgfpathmoveto{\pgfqpoint{0.682cm}{1.065cm}}
\pgfpathlineto{\pgfqpoint{1.246cm}{0.315cm}}
\pgfpathlineto{\pgfqpoint{1.811cm}{1.065cm}}
\pgfusepath{stroke}
\end{pgfscope}
\pgfpathmoveto{\pgfqpoint{1.948cm}{1.065cm}}
\pgfpathcurveto{\pgfqpoint{1.948cm}{1.101cm}}{\pgfqpoint{1.933cm}{1.136cm}}{\pgfqpoint{1.907cm}{1.162cm}}
\pgfpathcurveto{\pgfqpoint{1.882cm}{1.187cm}}{\pgfqpoint{1.847cm}{1.202cm}}{\pgfqpoint{1.811cm}{1.202cm}}
\pgfpathcurveto{\pgfqpoint{1.775cm}{1.202cm}}{\pgfqpoint{1.74cm}{1.187cm}}{\pgfqpoint{1.714cm}{1.162cm}}
\pgfpathcurveto{\pgfqpoint{1.689cm}{1.136cm}}{\pgfqpoint{1.674cm}{1.101cm}}{\pgfqpoint{1.674cm}{1.065cm}}
\pgfpathcurveto{\pgfqpoint{1.674cm}{1.029cm}}{\pgfqpoint{1.689cm}{0.994cm}}{\pgfqpoint{1.714cm}{0.968cm}}
\pgfpathcurveto{\pgfqpoint{1.74cm}{0.942cm}}{\pgfqpoint{1.775cm}{0.928cm}}{\pgfqpoint{1.811cm}{0.928cm}}
\pgfpathcurveto{\pgfqpoint{1.847cm}{0.928cm}}{\pgfqpoint{1.882cm}{0.942cm}}{\pgfqpoint{1.907cm}{0.968cm}}
\pgfpathcurveto{\pgfqpoint{1.933cm}{0.994cm}}{\pgfqpoint{1.948cm}{1.029cm}}{\pgfqpoint{1.948cm}{1.065cm}}
\pgfusepath{fill}
\begin{pgfscope}
\pgfsetdash{}{0cm}
\pgfsetlinewidth{0.818mm}
\pgfsetmiterlimit{7.0}
\pgfpathmoveto{\pgfqpoint{1.246cm}{0.315cm}}
\pgfpathlineto{\pgfqpoint{1.244cm}{1.061cm}}
\pgfusepath{stroke}
\end{pgfscope}
\pgfpathmoveto{\pgfqpoint{1.38cm}{1.065cm}}
\pgfpathcurveto{\pgfqpoint{1.38cm}{1.101cm}}{\pgfqpoint{1.366cm}{1.136cm}}{\pgfqpoint{1.34cm}{1.162cm}}
\pgfpathcurveto{\pgfqpoint{1.315cm}{1.187cm}}{\pgfqpoint{1.28cm}{1.202cm}}{\pgfqpoint{1.244cm}{1.202cm}}
\pgfpathcurveto{\pgfqpoint{1.207cm}{1.202cm}}{\pgfqpoint{1.173cm}{1.187cm}}{\pgfqpoint{1.147cm}{1.162cm}}
\pgfpathcurveto{\pgfqpoint{1.121cm}{1.136cm}}{\pgfqpoint{1.107cm}{1.101cm}}{\pgfqpoint{1.107cm}{1.065cm}}
\pgfpathcurveto{\pgfqpoint{1.107cm}{1.029cm}}{\pgfqpoint{1.121cm}{0.994cm}}{\pgfqpoint{1.147cm}{0.968cm}}
\pgfpathcurveto{\pgfqpoint{1.173cm}{0.942cm}}{\pgfqpoint{1.207cm}{0.928cm}}{\pgfqpoint{1.244cm}{0.928cm}}
\pgfpathcurveto{\pgfqpoint{1.28cm}{0.928cm}}{\pgfqpoint{1.315cm}{0.942cm}}{\pgfqpoint{1.34cm}{0.968cm}}
\pgfpathcurveto{\pgfqpoint{1.366cm}{0.994cm}}{\pgfqpoint{1.38cm}{1.029cm}}{\pgfqpoint{1.38cm}{1.065cm}}
\pgfusepath{fill}
\begin{pgfscope}
\pgfsetdash{}{0cm}
\pgfsetlinewidth{0.818mm}
\pgfsetmiterlimit{4.0}
\pgfpathmoveto{\pgfqpoint{1.383cm}{0.178cm}}
\pgfpathcurveto{\pgfqpoint{1.383cm}{0.214cm}}{\pgfqpoint{1.369cm}{0.249cm}}{\pgfqpoint{1.343cm}{0.275cm}}
\pgfpathcurveto{\pgfqpoint{1.317cm}{0.3cm}}{\pgfqpoint{1.283cm}{0.315cm}}{\pgfqpoint{1.246cm}{0.315cm}}
\pgfpathcurveto{\pgfqpoint{1.21cm}{0.315cm}}{\pgfqpoint{1.175cm}{0.3cm}}{\pgfqpoint{1.15cm}{0.275cm}}
\pgfpathcurveto{\pgfqpoint{1.124cm}{0.249cm}}{\pgfqpoint{1.11cm}{0.214cm}}{\pgfqpoint{1.11cm}{0.178cm}}
\pgfpathcurveto{\pgfqpoint{1.11cm}{0.141cm}}{\pgfqpoint{1.124cm}{0.107cm}}{\pgfqpoint{1.15cm}{0.081cm}}
\pgfpathcurveto{\pgfqpoint{1.175cm}{0.055cm}}{\pgfqpoint{1.21cm}{0.041cm}}{\pgfqpoint{1.246cm}{0.041cm}}
\pgfpathcurveto{\pgfqpoint{1.283cm}{0.041cm}}{\pgfqpoint{1.317cm}{0.055cm}}{\pgfqpoint{1.343cm}{0.081cm}}
\pgfpathcurveto{\pgfqpoint{1.369cm}{0.107cm}}{\pgfqpoint{1.383cm}{0.141cm}}{\pgfqpoint{1.383cm}{0.178cm}}
\pgfusepath{stroke}
\end{pgfscope}
\end{pgfscope}
\end{pgfscope}
\end{pgfscope}
\end{tikzpicture}}} \|_{\CC^{-\kappa}(\rho^\sigma)},
\end{equation}
\begin{align*}
&\|6\UU_>X\succ(X^{\!\resizebox{0.6em}{!}{
\begin{tikzpicture}
\pgfpathmoveto{\pgfqpoint{0cm}{-0.035cm}}
\pgfpathlineto{\pgfqpoint{1.376cm}{-0.035cm}}
\pgfpathlineto{\pgfqpoint{1.376cm}{1.552cm}}
\pgfpathlineto{\pgfqpoint{0cm}{1.552cm}}
\pgfpathclose
\pgfusepath{clip}
\begin{pgfscope}
\begin{pgfscope}
\pgfpathmoveto{\pgfqpoint{0cm}{-0.035cm}}
\pgfpathlineto{\pgfqpoint{1.376cm}{-0.035cm}}
\pgfpathlineto{\pgfqpoint{1.376cm}{1.552cm}}
\pgfpathlineto{\pgfqpoint{0cm}{1.552cm}}
\pgfpathclose
\pgfusepath{clip}
\begin{pgfscope}
\begin{pgfscope}
\pgfsetdash{}{0cm}
\pgfsetlinewidth{0.818mm}
\pgfsetroundcap
\pgfsetroundjoin
\pgfsetmiterlimit{7.0}
\definecolor{eps2pgf_color}{gray}{0}\pgfsetstrokecolor{eps2pgf_color}\pgfsetfillcolor{eps2pgf_color}
\pgfpathmoveto{\pgfqpoint{0.117cm}{1.421cm}}
\pgfpathlineto{\pgfqpoint{0.682cm}{0.671cm}}
\pgfpathlineto{\pgfqpoint{1.246cm}{1.421cm}}
\pgfusepath{stroke}
\end{pgfscope}
\definecolor{eps2pgf_color}{gray}{0}\pgfsetstrokecolor{eps2pgf_color}\pgfsetfillcolor{eps2pgf_color}
\pgfpathmoveto{\pgfqpoint{0.273cm}{1.395cm}}
\pgfpathcurveto{\pgfqpoint{0.273cm}{1.432cm}}{\pgfqpoint{0.259cm}{1.467cm}}{\pgfqpoint{0.233cm}{1.492cm}}
\pgfpathcurveto{\pgfqpoint{0.207cm}{1.518cm}}{\pgfqpoint{0.173cm}{1.532cm}}{\pgfqpoint{0.137cm}{1.532cm}}
\pgfpathcurveto{\pgfqpoint{0.1cm}{1.532cm}}{\pgfqpoint{0.066cm}{1.518cm}}{\pgfqpoint{0.04cm}{1.492cm}}
\pgfpathcurveto{\pgfqpoint{0.014cm}{1.467cm}}{\pgfqpoint{0cm}{1.432cm}}{\pgfqpoint{0cm}{1.395cm}}
\pgfpathcurveto{\pgfqpoint{0cm}{1.359cm}}{\pgfqpoint{0.014cm}{1.324cm}}{\pgfqpoint{0.04cm}{1.299cm}}
\pgfpathcurveto{\pgfqpoint{0.066cm}{1.273cm}}{\pgfqpoint{0.1cm}{1.258cm}}{\pgfqpoint{0.137cm}{1.258cm}}
\pgfpathcurveto{\pgfqpoint{0.173cm}{1.258cm}}{\pgfqpoint{0.207cm}{1.273cm}}{\pgfqpoint{0.233cm}{1.299cm}}
\pgfpathcurveto{\pgfqpoint{0.259cm}{1.324cm}}{\pgfqpoint{0.273cm}{1.359cm}}{\pgfqpoint{0.273cm}{1.395cm}}
\pgfusepath{fill}
\begin{pgfscope}
\pgfsetdash{}{0cm}
\pgfsetlinewidth{0.818mm}
\pgfsetmiterlimit{7.0}
\pgfpathmoveto{\pgfqpoint{0.682cm}{0.671cm}}
\pgfpathlineto{\pgfqpoint{0.679cm}{1.418cm}}
\pgfusepath{stroke}
\end{pgfscope}
\pgfpathmoveto{\pgfqpoint{0.815cm}{1.399cm}}
\pgfpathcurveto{\pgfqpoint{0.815cm}{1.435cm}}{\pgfqpoint{0.801cm}{1.47cm}}{\pgfqpoint{0.775cm}{1.496cm}}
\pgfpathcurveto{\pgfqpoint{0.75cm}{1.521cm}}{\pgfqpoint{0.715cm}{1.536cm}}{\pgfqpoint{0.679cm}{1.536cm}}
\pgfpathcurveto{\pgfqpoint{0.643cm}{1.536cm}}{\pgfqpoint{0.608cm}{1.521cm}}{\pgfqpoint{0.582cm}{1.496cm}}
\pgfpathcurveto{\pgfqpoint{0.557cm}{1.47cm}}{\pgfqpoint{0.542cm}{1.435cm}}{\pgfqpoint{0.542cm}{1.399cm}}
\pgfpathcurveto{\pgfqpoint{0.542cm}{1.363cm}}{\pgfqpoint{0.557cm}{1.328cm}}{\pgfqpoint{0.582cm}{1.302cm}}
\pgfpathcurveto{\pgfqpoint{0.608cm}{1.276cm}}{\pgfqpoint{0.643cm}{1.262cm}}{\pgfqpoint{0.679cm}{1.262cm}}
\pgfpathcurveto{\pgfqpoint{0.715cm}{1.262cm}}{\pgfqpoint{0.75cm}{1.276cm}}{\pgfqpoint{0.775cm}{1.302cm}}
\pgfpathcurveto{\pgfqpoint{0.801cm}{1.328cm}}{\pgfqpoint{0.815cm}{1.363cm}}{\pgfqpoint{0.815cm}{1.399cm}}
\pgfusepath{fill}
\pgfpathmoveto{\pgfqpoint{1.345cm}{1.371cm}}
\pgfpathcurveto{\pgfqpoint{1.345cm}{1.408cm}}{\pgfqpoint{1.331cm}{1.442cm}}{\pgfqpoint{1.305cm}{1.468cm}}
\pgfpathcurveto{\pgfqpoint{1.28cm}{1.494cm}}{\pgfqpoint{1.245cm}{1.508cm}}{\pgfqpoint{1.209cm}{1.508cm}}
\pgfpathcurveto{\pgfqpoint{1.172cm}{1.508cm}}{\pgfqpoint{1.138cm}{1.494cm}}{\pgfqpoint{1.112cm}{1.468cm}}
\pgfpathcurveto{\pgfqpoint{1.087cm}{1.442cm}}{\pgfqpoint{1.072cm}{1.408cm}}{\pgfqpoint{1.072cm}{1.371cm}}
\pgfpathcurveto{\pgfqpoint{1.072cm}{1.335cm}}{\pgfqpoint{1.087cm}{1.3cm}}{\pgfqpoint{1.112cm}{1.274cm}}
\pgfpathcurveto{\pgfqpoint{1.138cm}{1.249cm}}{\pgfqpoint{1.172cm}{1.234cm}}{\pgfqpoint{1.209cm}{1.234cm}}
\pgfpathcurveto{\pgfqpoint{1.245cm}{1.234cm}}{\pgfqpoint{1.28cm}{1.249cm}}{\pgfqpoint{1.305cm}{1.274cm}}
\pgfpathcurveto{\pgfqpoint{1.331cm}{1.3cm}}{\pgfqpoint{1.345cm}{1.335cm}}{\pgfqpoint{1.345cm}{1.371cm}}
\pgfusepath{fill}
\begin{pgfscope}
\pgfsetdash{}{0cm}
\pgfsetlinewidth{0.818mm}
\pgfsetroundcap
\pgfsetmiterlimit{4.0}
\pgfpathmoveto{\pgfqpoint{0.682cm}{0.671cm}}
\pgfpathlineto{\pgfqpoint{0.682cm}{0.042cm}}
\pgfusepath{stroke}
\end{pgfscope}
\end{pgfscope}
\end{pgfscope}
\end{pgfscope}
\end{tikzpicture}}}(\phi+\psi))\|_{\CC^{-2+\alpha}(\rho)}+\|6\UU_> X\prec(X^{\!\resizebox{0.6em}{!}{
\begin{tikzpicture}
\pgfpathmoveto{\pgfqpoint{0cm}{-0.035cm}}
\pgfpathlineto{\pgfqpoint{1.376cm}{-0.035cm}}
\pgfpathlineto{\pgfqpoint{1.376cm}{1.552cm}}
\pgfpathlineto{\pgfqpoint{0cm}{1.552cm}}
\pgfpathclose
\pgfusepath{clip}
\begin{pgfscope}
\begin{pgfscope}
\pgfpathmoveto{\pgfqpoint{0cm}{-0.035cm}}
\pgfpathlineto{\pgfqpoint{1.376cm}{-0.035cm}}
\pgfpathlineto{\pgfqpoint{1.376cm}{1.552cm}}
\pgfpathlineto{\pgfqpoint{0cm}{1.552cm}}
\pgfpathclose
\pgfusepath{clip}
\begin{pgfscope}
\begin{pgfscope}
\pgfsetdash{}{0cm}
\pgfsetlinewidth{0.818mm}
\pgfsetroundcap
\pgfsetroundjoin
\pgfsetmiterlimit{7.0}
\definecolor{eps2pgf_color}{gray}{0}\pgfsetstrokecolor{eps2pgf_color}\pgfsetfillcolor{eps2pgf_color}
\pgfpathmoveto{\pgfqpoint{0.117cm}{1.421cm}}
\pgfpathlineto{\pgfqpoint{0.682cm}{0.671cm}}
\pgfpathlineto{\pgfqpoint{1.246cm}{1.421cm}}
\pgfusepath{stroke}
\end{pgfscope}
\definecolor{eps2pgf_color}{gray}{0}\pgfsetstrokecolor{eps2pgf_color}\pgfsetfillcolor{eps2pgf_color}
\pgfpathmoveto{\pgfqpoint{0.273cm}{1.395cm}}
\pgfpathcurveto{\pgfqpoint{0.273cm}{1.432cm}}{\pgfqpoint{0.259cm}{1.467cm}}{\pgfqpoint{0.233cm}{1.492cm}}
\pgfpathcurveto{\pgfqpoint{0.207cm}{1.518cm}}{\pgfqpoint{0.173cm}{1.532cm}}{\pgfqpoint{0.137cm}{1.532cm}}
\pgfpathcurveto{\pgfqpoint{0.1cm}{1.532cm}}{\pgfqpoint{0.066cm}{1.518cm}}{\pgfqpoint{0.04cm}{1.492cm}}
\pgfpathcurveto{\pgfqpoint{0.014cm}{1.467cm}}{\pgfqpoint{0cm}{1.432cm}}{\pgfqpoint{0cm}{1.395cm}}
\pgfpathcurveto{\pgfqpoint{0cm}{1.359cm}}{\pgfqpoint{0.014cm}{1.324cm}}{\pgfqpoint{0.04cm}{1.299cm}}
\pgfpathcurveto{\pgfqpoint{0.066cm}{1.273cm}}{\pgfqpoint{0.1cm}{1.258cm}}{\pgfqpoint{0.137cm}{1.258cm}}
\pgfpathcurveto{\pgfqpoint{0.173cm}{1.258cm}}{\pgfqpoint{0.207cm}{1.273cm}}{\pgfqpoint{0.233cm}{1.299cm}}
\pgfpathcurveto{\pgfqpoint{0.259cm}{1.324cm}}{\pgfqpoint{0.273cm}{1.359cm}}{\pgfqpoint{0.273cm}{1.395cm}}
\pgfusepath{fill}
\begin{pgfscope}
\pgfsetdash{}{0cm}
\pgfsetlinewidth{0.818mm}
\pgfsetmiterlimit{7.0}
\pgfpathmoveto{\pgfqpoint{0.682cm}{0.671cm}}
\pgfpathlineto{\pgfqpoint{0.679cm}{1.418cm}}
\pgfusepath{stroke}
\end{pgfscope}
\pgfpathmoveto{\pgfqpoint{0.815cm}{1.399cm}}
\pgfpathcurveto{\pgfqpoint{0.815cm}{1.435cm}}{\pgfqpoint{0.801cm}{1.47cm}}{\pgfqpoint{0.775cm}{1.496cm}}
\pgfpathcurveto{\pgfqpoint{0.75cm}{1.521cm}}{\pgfqpoint{0.715cm}{1.536cm}}{\pgfqpoint{0.679cm}{1.536cm}}
\pgfpathcurveto{\pgfqpoint{0.643cm}{1.536cm}}{\pgfqpoint{0.608cm}{1.521cm}}{\pgfqpoint{0.582cm}{1.496cm}}
\pgfpathcurveto{\pgfqpoint{0.557cm}{1.47cm}}{\pgfqpoint{0.542cm}{1.435cm}}{\pgfqpoint{0.542cm}{1.399cm}}
\pgfpathcurveto{\pgfqpoint{0.542cm}{1.363cm}}{\pgfqpoint{0.557cm}{1.328cm}}{\pgfqpoint{0.582cm}{1.302cm}}
\pgfpathcurveto{\pgfqpoint{0.608cm}{1.276cm}}{\pgfqpoint{0.643cm}{1.262cm}}{\pgfqpoint{0.679cm}{1.262cm}}
\pgfpathcurveto{\pgfqpoint{0.715cm}{1.262cm}}{\pgfqpoint{0.75cm}{1.276cm}}{\pgfqpoint{0.775cm}{1.302cm}}
\pgfpathcurveto{\pgfqpoint{0.801cm}{1.328cm}}{\pgfqpoint{0.815cm}{1.363cm}}{\pgfqpoint{0.815cm}{1.399cm}}
\pgfusepath{fill}
\pgfpathmoveto{\pgfqpoint{1.345cm}{1.371cm}}
\pgfpathcurveto{\pgfqpoint{1.345cm}{1.408cm}}{\pgfqpoint{1.331cm}{1.442cm}}{\pgfqpoint{1.305cm}{1.468cm}}
\pgfpathcurveto{\pgfqpoint{1.28cm}{1.494cm}}{\pgfqpoint{1.245cm}{1.508cm}}{\pgfqpoint{1.209cm}{1.508cm}}
\pgfpathcurveto{\pgfqpoint{1.172cm}{1.508cm}}{\pgfqpoint{1.138cm}{1.494cm}}{\pgfqpoint{1.112cm}{1.468cm}}
\pgfpathcurveto{\pgfqpoint{1.087cm}{1.442cm}}{\pgfqpoint{1.072cm}{1.408cm}}{\pgfqpoint{1.072cm}{1.371cm}}
\pgfpathcurveto{\pgfqpoint{1.072cm}{1.335cm}}{\pgfqpoint{1.087cm}{1.3cm}}{\pgfqpoint{1.112cm}{1.274cm}}
\pgfpathcurveto{\pgfqpoint{1.138cm}{1.249cm}}{\pgfqpoint{1.172cm}{1.234cm}}{\pgfqpoint{1.209cm}{1.234cm}}
\pgfpathcurveto{\pgfqpoint{1.245cm}{1.234cm}}{\pgfqpoint{1.28cm}{1.249cm}}{\pgfqpoint{1.305cm}{1.274cm}}
\pgfpathcurveto{\pgfqpoint{1.331cm}{1.3cm}}{\pgfqpoint{1.345cm}{1.335cm}}{\pgfqpoint{1.345cm}{1.371cm}}
\pgfusepath{fill}
\begin{pgfscope}
\pgfsetdash{}{0cm}
\pgfsetlinewidth{0.818mm}
\pgfsetroundcap
\pgfsetmiterlimit{4.0}
\pgfpathmoveto{\pgfqpoint{0.682cm}{0.671cm}}
\pgfpathlineto{\pgfqpoint{0.682cm}{0.042cm}}
\pgfusepath{stroke}
\end{pgfscope}
\end{pgfscope}
\end{pgfscope}
\end{pgfscope}
\end{tikzpicture}}}(\phi+\psi))\|_{\CC^{-2+\alpha}(\rho)}\\
&\quad\lesssim\|\phi+\psi\|_{L^\infty(\rho)}\|\UU_>X\|_{\CC^{-2+\alpha}(\rho^{-\alpha})}\\
&\quad\lesssim 2^{-(\frac{3}{2}-\alpha-\kappa)\frac{2}{3}K}\|\phi+\psi\|_{L^\infty(\rho)},
\end{align*}
provided
\begin{equation}\label{eq:l3}
\|\UU_> X\|_{\CC^{-2+\alpha}(\rho^{-\alpha})}\lesssim 2^{-(\frac{3}{2}-\alpha-\kappa)\frac{2}{3}K} \|X\|_{\CC^{-\frac{1}{2}-\kappa}(\rho^\sigma)},
\end{equation}
\begin{align*}
\|6(\phi+\psi)\prec\UU_{>}X^{\!\resizebox{!}{.8em}{
\begin{tikzpicture}
\pgfpathmoveto{\pgfqpoint{0cm}{-0.035cm}}
\pgfpathlineto{\pgfqpoint{1.976cm}{-0.035cm}}
\pgfpathlineto{\pgfqpoint{1.976cm}{1.94cm}}
\pgfpathlineto{\pgfqpoint{0cm}{1.94cm}}
\pgfpathclose
\pgfusepath{clip}
\begin{pgfscope}
\begin{pgfscope}
\pgfpathmoveto{\pgfqpoint{0cm}{-0.035cm}}
\pgfpathlineto{\pgfqpoint{1.976cm}{-0.035cm}}
\pgfpathlineto{\pgfqpoint{1.976cm}{1.94cm}}
\pgfpathlineto{\pgfqpoint{0cm}{1.94cm}}
\pgfpathclose
\pgfusepath{clip}
\begin{pgfscope}
\begin{pgfscope}
\pgfsetdash{}{0cm}
\pgfsetlinewidth{0.818mm}
\pgfsetroundcap
\pgfsetroundjoin
\pgfsetmiterlimit{7.0}
\definecolor{eps2pgf_color}{gray}{0}\pgfsetstrokecolor{eps2pgf_color}\pgfsetfillcolor{eps2pgf_color}
\pgfpathmoveto{\pgfqpoint{0.117cm}{1.815cm}}
\pgfpathlineto{\pgfqpoint{0.682cm}{1.065cm}}
\pgfpathlineto{\pgfqpoint{1.246cm}{1.815cm}}
\pgfusepath{stroke}
\end{pgfscope}
\definecolor{eps2pgf_color}{gray}{0}\pgfsetstrokecolor{eps2pgf_color}\pgfsetfillcolor{eps2pgf_color}
\pgfpathmoveto{\pgfqpoint{0.273cm}{1.789cm}}
\pgfpathcurveto{\pgfqpoint{0.273cm}{1.825cm}}{\pgfqpoint{0.259cm}{1.86cm}}{\pgfqpoint{0.233cm}{1.886cm}}
\pgfpathcurveto{\pgfqpoint{0.207cm}{1.912cm}}{\pgfqpoint{0.173cm}{1.926cm}}{\pgfqpoint{0.137cm}{1.926cm}}
\pgfpathcurveto{\pgfqpoint{0.1cm}{1.926cm}}{\pgfqpoint{0.066cm}{1.912cm}}{\pgfqpoint{0.04cm}{1.886cm}}
\pgfpathcurveto{\pgfqpoint{0.014cm}{1.86cm}}{\pgfqpoint{0cm}{1.825cm}}{\pgfqpoint{0cm}{1.789cm}}
\pgfpathcurveto{\pgfqpoint{0cm}{1.753cm}}{\pgfqpoint{0.014cm}{1.718cm}}{\pgfqpoint{0.04cm}{1.692cm}}
\pgfpathcurveto{\pgfqpoint{0.066cm}{1.667cm}}{\pgfqpoint{0.1cm}{1.652cm}}{\pgfqpoint{0.137cm}{1.652cm}}
\pgfpathcurveto{\pgfqpoint{0.173cm}{1.652cm}}{\pgfqpoint{0.207cm}{1.667cm}}{\pgfqpoint{0.233cm}{1.692cm}}
\pgfpathcurveto{\pgfqpoint{0.259cm}{1.718cm}}{\pgfqpoint{0.273cm}{1.753cm}}{\pgfqpoint{0.273cm}{1.789cm}}
\pgfusepath{fill}
\begin{pgfscope}
\pgfsetdash{}{0cm}
\pgfsetlinewidth{0.818mm}
\pgfsetmiterlimit{7.0}
\pgfpathmoveto{\pgfqpoint{0.682cm}{1.065cm}}
\pgfpathlineto{\pgfqpoint{0.679cm}{1.812cm}}
\pgfusepath{stroke}
\end{pgfscope}
\pgfpathmoveto{\pgfqpoint{0.815cm}{1.793cm}}
\pgfpathcurveto{\pgfqpoint{0.815cm}{1.829cm}}{\pgfqpoint{0.801cm}{1.864cm}}{\pgfqpoint{0.775cm}{1.89cm}}
\pgfpathcurveto{\pgfqpoint{0.75cm}{1.915cm}}{\pgfqpoint{0.715cm}{1.93cm}}{\pgfqpoint{0.679cm}{1.93cm}}
\pgfpathcurveto{\pgfqpoint{0.643cm}{1.93cm}}{\pgfqpoint{0.608cm}{1.915cm}}{\pgfqpoint{0.582cm}{1.89cm}}
\pgfpathcurveto{\pgfqpoint{0.557cm}{1.864cm}}{\pgfqpoint{0.542cm}{1.829cm}}{\pgfqpoint{0.542cm}{1.793cm}}
\pgfpathcurveto{\pgfqpoint{0.542cm}{1.756cm}}{\pgfqpoint{0.557cm}{1.722cm}}{\pgfqpoint{0.582cm}{1.696cm}}
\pgfpathcurveto{\pgfqpoint{0.608cm}{1.67cm}}{\pgfqpoint{0.643cm}{1.656cm}}{\pgfqpoint{0.679cm}{1.656cm}}
\pgfpathcurveto{\pgfqpoint{0.715cm}{1.656cm}}{\pgfqpoint{0.75cm}{1.67cm}}{\pgfqpoint{0.775cm}{1.696cm}}
\pgfpathcurveto{\pgfqpoint{0.801cm}{1.722cm}}{\pgfqpoint{0.815cm}{1.756cm}}{\pgfqpoint{0.815cm}{1.793cm}}
\pgfusepath{fill}
\pgfpathmoveto{\pgfqpoint{1.345cm}{1.765cm}}
\pgfpathcurveto{\pgfqpoint{1.345cm}{1.801cm}}{\pgfqpoint{1.331cm}{1.836cm}}{\pgfqpoint{1.305cm}{1.862cm}}
\pgfpathcurveto{\pgfqpoint{1.28cm}{1.887cm}}{\pgfqpoint{1.245cm}{1.902cm}}{\pgfqpoint{1.209cm}{1.902cm}}
\pgfpathcurveto{\pgfqpoint{1.172cm}{1.902cm}}{\pgfqpoint{1.138cm}{1.887cm}}{\pgfqpoint{1.112cm}{1.862cm}}
\pgfpathcurveto{\pgfqpoint{1.087cm}{1.836cm}}{\pgfqpoint{1.072cm}{1.801cm}}{\pgfqpoint{1.072cm}{1.765cm}}
\pgfpathcurveto{\pgfqpoint{1.072cm}{1.728cm}}{\pgfqpoint{1.087cm}{1.694cm}}{\pgfqpoint{1.112cm}{1.668cm}}
\pgfpathcurveto{\pgfqpoint{1.138cm}{1.642cm}}{\pgfqpoint{1.172cm}{1.628cm}}{\pgfqpoint{1.209cm}{1.628cm}}
\pgfpathcurveto{\pgfqpoint{1.245cm}{1.628cm}}{\pgfqpoint{1.28cm}{1.642cm}}{\pgfqpoint{1.305cm}{1.668cm}}
\pgfpathcurveto{\pgfqpoint{1.331cm}{1.694cm}}{\pgfqpoint{1.345cm}{1.728cm}}{\pgfqpoint{1.345cm}{1.765cm}}
\pgfusepath{fill}
\begin{pgfscope}
\pgfsetdash{}{0cm}
\pgfsetlinewidth{0.818mm}
\pgfsetroundcap
\pgfsetroundjoin
\pgfsetmiterlimit{7.0}
\pgfpathmoveto{\pgfqpoint{0.682cm}{1.065cm}}
\pgfpathlineto{\pgfqpoint{1.246cm}{0.315cm}}
\pgfpathlineto{\pgfqpoint{1.811cm}{1.065cm}}
\pgfusepath{stroke}
\end{pgfscope}
\pgfpathmoveto{\pgfqpoint{1.948cm}{1.065cm}}
\pgfpathcurveto{\pgfqpoint{1.948cm}{1.101cm}}{\pgfqpoint{1.933cm}{1.136cm}}{\pgfqpoint{1.907cm}{1.162cm}}
\pgfpathcurveto{\pgfqpoint{1.882cm}{1.187cm}}{\pgfqpoint{1.847cm}{1.202cm}}{\pgfqpoint{1.811cm}{1.202cm}}
\pgfpathcurveto{\pgfqpoint{1.775cm}{1.202cm}}{\pgfqpoint{1.74cm}{1.187cm}}{\pgfqpoint{1.714cm}{1.162cm}}
\pgfpathcurveto{\pgfqpoint{1.689cm}{1.136cm}}{\pgfqpoint{1.674cm}{1.101cm}}{\pgfqpoint{1.674cm}{1.065cm}}
\pgfpathcurveto{\pgfqpoint{1.674cm}{1.029cm}}{\pgfqpoint{1.689cm}{0.994cm}}{\pgfqpoint{1.714cm}{0.968cm}}
\pgfpathcurveto{\pgfqpoint{1.74cm}{0.942cm}}{\pgfqpoint{1.775cm}{0.928cm}}{\pgfqpoint{1.811cm}{0.928cm}}
\pgfpathcurveto{\pgfqpoint{1.847cm}{0.928cm}}{\pgfqpoint{1.882cm}{0.942cm}}{\pgfqpoint{1.907cm}{0.968cm}}
\pgfpathcurveto{\pgfqpoint{1.933cm}{0.994cm}}{\pgfqpoint{1.948cm}{1.029cm}}{\pgfqpoint{1.948cm}{1.065cm}}
\pgfusepath{fill}
\begin{pgfscope}
\pgfsetdash{}{0cm}
\pgfsetlinewidth{0.818mm}
\pgfsetmiterlimit{4.0}
\pgfpathmoveto{\pgfqpoint{1.383cm}{0.178cm}}
\pgfpathcurveto{\pgfqpoint{1.383cm}{0.214cm}}{\pgfqpoint{1.369cm}{0.249cm}}{\pgfqpoint{1.343cm}{0.275cm}}
\pgfpathcurveto{\pgfqpoint{1.317cm}{0.3cm}}{\pgfqpoint{1.283cm}{0.315cm}}{\pgfqpoint{1.246cm}{0.315cm}}
\pgfpathcurveto{\pgfqpoint{1.21cm}{0.315cm}}{\pgfqpoint{1.175cm}{0.3cm}}{\pgfqpoint{1.15cm}{0.275cm}}
\pgfpathcurveto{\pgfqpoint{1.124cm}{0.249cm}}{\pgfqpoint{1.11cm}{0.214cm}}{\pgfqpoint{1.11cm}{0.178cm}}
\pgfpathcurveto{\pgfqpoint{1.11cm}{0.141cm}}{\pgfqpoint{1.124cm}{0.107cm}}{\pgfqpoint{1.15cm}{0.081cm}}
\pgfpathcurveto{\pgfqpoint{1.175cm}{0.055cm}}{\pgfqpoint{1.21cm}{0.041cm}}{\pgfqpoint{1.246cm}{0.041cm}}
\pgfpathcurveto{\pgfqpoint{1.283cm}{0.041cm}}{\pgfqpoint{1.317cm}{0.055cm}}{\pgfqpoint{1.343cm}{0.081cm}}
\pgfpathcurveto{\pgfqpoint{1.369cm}{0.107cm}}{\pgfqpoint{1.383cm}{0.141cm}}{\pgfqpoint{1.383cm}{0.178cm}}
\pgfusepath{stroke}
\end{pgfscope}
\end{pgfscope}
\end{pgfscope}
\end{pgfscope}
\end{tikzpicture}}}\|_{\CC^{-2+\alpha}(\rho)}
&\lesssim \|\phi+\psi\|_{L^\infty(\rho)}\|\UU_>X^{\!\resizebox{!}{.8em}{
\begin{tikzpicture}
\pgfpathmoveto{\pgfqpoint{0cm}{-0.035cm}}
\pgfpathlineto{\pgfqpoint{1.976cm}{-0.035cm}}
\pgfpathlineto{\pgfqpoint{1.976cm}{1.94cm}}
\pgfpathlineto{\pgfqpoint{0cm}{1.94cm}}
\pgfpathclose
\pgfusepath{clip}
\begin{pgfscope}
\begin{pgfscope}
\pgfpathmoveto{\pgfqpoint{0cm}{-0.035cm}}
\pgfpathlineto{\pgfqpoint{1.976cm}{-0.035cm}}
\pgfpathlineto{\pgfqpoint{1.976cm}{1.94cm}}
\pgfpathlineto{\pgfqpoint{0cm}{1.94cm}}
\pgfpathclose
\pgfusepath{clip}
\begin{pgfscope}
\begin{pgfscope}
\pgfsetdash{}{0cm}
\pgfsetlinewidth{0.818mm}
\pgfsetroundcap
\pgfsetroundjoin
\pgfsetmiterlimit{7.0}
\definecolor{eps2pgf_color}{gray}{0}\pgfsetstrokecolor{eps2pgf_color}\pgfsetfillcolor{eps2pgf_color}
\pgfpathmoveto{\pgfqpoint{0.117cm}{1.815cm}}
\pgfpathlineto{\pgfqpoint{0.682cm}{1.065cm}}
\pgfpathlineto{\pgfqpoint{1.246cm}{1.815cm}}
\pgfusepath{stroke}
\end{pgfscope}
\definecolor{eps2pgf_color}{gray}{0}\pgfsetstrokecolor{eps2pgf_color}\pgfsetfillcolor{eps2pgf_color}
\pgfpathmoveto{\pgfqpoint{0.273cm}{1.789cm}}
\pgfpathcurveto{\pgfqpoint{0.273cm}{1.825cm}}{\pgfqpoint{0.259cm}{1.86cm}}{\pgfqpoint{0.233cm}{1.886cm}}
\pgfpathcurveto{\pgfqpoint{0.207cm}{1.912cm}}{\pgfqpoint{0.173cm}{1.926cm}}{\pgfqpoint{0.137cm}{1.926cm}}
\pgfpathcurveto{\pgfqpoint{0.1cm}{1.926cm}}{\pgfqpoint{0.066cm}{1.912cm}}{\pgfqpoint{0.04cm}{1.886cm}}
\pgfpathcurveto{\pgfqpoint{0.014cm}{1.86cm}}{\pgfqpoint{0cm}{1.825cm}}{\pgfqpoint{0cm}{1.789cm}}
\pgfpathcurveto{\pgfqpoint{0cm}{1.753cm}}{\pgfqpoint{0.014cm}{1.718cm}}{\pgfqpoint{0.04cm}{1.692cm}}
\pgfpathcurveto{\pgfqpoint{0.066cm}{1.667cm}}{\pgfqpoint{0.1cm}{1.652cm}}{\pgfqpoint{0.137cm}{1.652cm}}
\pgfpathcurveto{\pgfqpoint{0.173cm}{1.652cm}}{\pgfqpoint{0.207cm}{1.667cm}}{\pgfqpoint{0.233cm}{1.692cm}}
\pgfpathcurveto{\pgfqpoint{0.259cm}{1.718cm}}{\pgfqpoint{0.273cm}{1.753cm}}{\pgfqpoint{0.273cm}{1.789cm}}
\pgfusepath{fill}
\begin{pgfscope}
\pgfsetdash{}{0cm}
\pgfsetlinewidth{0.818mm}
\pgfsetmiterlimit{7.0}
\pgfpathmoveto{\pgfqpoint{0.682cm}{1.065cm}}
\pgfpathlineto{\pgfqpoint{0.679cm}{1.812cm}}
\pgfusepath{stroke}
\end{pgfscope}
\pgfpathmoveto{\pgfqpoint{0.815cm}{1.793cm}}
\pgfpathcurveto{\pgfqpoint{0.815cm}{1.829cm}}{\pgfqpoint{0.801cm}{1.864cm}}{\pgfqpoint{0.775cm}{1.89cm}}
\pgfpathcurveto{\pgfqpoint{0.75cm}{1.915cm}}{\pgfqpoint{0.715cm}{1.93cm}}{\pgfqpoint{0.679cm}{1.93cm}}
\pgfpathcurveto{\pgfqpoint{0.643cm}{1.93cm}}{\pgfqpoint{0.608cm}{1.915cm}}{\pgfqpoint{0.582cm}{1.89cm}}
\pgfpathcurveto{\pgfqpoint{0.557cm}{1.864cm}}{\pgfqpoint{0.542cm}{1.829cm}}{\pgfqpoint{0.542cm}{1.793cm}}
\pgfpathcurveto{\pgfqpoint{0.542cm}{1.756cm}}{\pgfqpoint{0.557cm}{1.722cm}}{\pgfqpoint{0.582cm}{1.696cm}}
\pgfpathcurveto{\pgfqpoint{0.608cm}{1.67cm}}{\pgfqpoint{0.643cm}{1.656cm}}{\pgfqpoint{0.679cm}{1.656cm}}
\pgfpathcurveto{\pgfqpoint{0.715cm}{1.656cm}}{\pgfqpoint{0.75cm}{1.67cm}}{\pgfqpoint{0.775cm}{1.696cm}}
\pgfpathcurveto{\pgfqpoint{0.801cm}{1.722cm}}{\pgfqpoint{0.815cm}{1.756cm}}{\pgfqpoint{0.815cm}{1.793cm}}
\pgfusepath{fill}
\pgfpathmoveto{\pgfqpoint{1.345cm}{1.765cm}}
\pgfpathcurveto{\pgfqpoint{1.345cm}{1.801cm}}{\pgfqpoint{1.331cm}{1.836cm}}{\pgfqpoint{1.305cm}{1.862cm}}
\pgfpathcurveto{\pgfqpoint{1.28cm}{1.887cm}}{\pgfqpoint{1.245cm}{1.902cm}}{\pgfqpoint{1.209cm}{1.902cm}}
\pgfpathcurveto{\pgfqpoint{1.172cm}{1.902cm}}{\pgfqpoint{1.138cm}{1.887cm}}{\pgfqpoint{1.112cm}{1.862cm}}
\pgfpathcurveto{\pgfqpoint{1.087cm}{1.836cm}}{\pgfqpoint{1.072cm}{1.801cm}}{\pgfqpoint{1.072cm}{1.765cm}}
\pgfpathcurveto{\pgfqpoint{1.072cm}{1.728cm}}{\pgfqpoint{1.087cm}{1.694cm}}{\pgfqpoint{1.112cm}{1.668cm}}
\pgfpathcurveto{\pgfqpoint{1.138cm}{1.642cm}}{\pgfqpoint{1.172cm}{1.628cm}}{\pgfqpoint{1.209cm}{1.628cm}}
\pgfpathcurveto{\pgfqpoint{1.245cm}{1.628cm}}{\pgfqpoint{1.28cm}{1.642cm}}{\pgfqpoint{1.305cm}{1.668cm}}
\pgfpathcurveto{\pgfqpoint{1.331cm}{1.694cm}}{\pgfqpoint{1.345cm}{1.728cm}}{\pgfqpoint{1.345cm}{1.765cm}}
\pgfusepath{fill}
\begin{pgfscope}
\pgfsetdash{}{0cm}
\pgfsetlinewidth{0.818mm}
\pgfsetroundcap
\pgfsetroundjoin
\pgfsetmiterlimit{7.0}
\pgfpathmoveto{\pgfqpoint{0.682cm}{1.065cm}}
\pgfpathlineto{\pgfqpoint{1.246cm}{0.315cm}}
\pgfpathlineto{\pgfqpoint{1.811cm}{1.065cm}}
\pgfusepath{stroke}
\end{pgfscope}
\pgfpathmoveto{\pgfqpoint{1.948cm}{1.065cm}}
\pgfpathcurveto{\pgfqpoint{1.948cm}{1.101cm}}{\pgfqpoint{1.933cm}{1.136cm}}{\pgfqpoint{1.907cm}{1.162cm}}
\pgfpathcurveto{\pgfqpoint{1.882cm}{1.187cm}}{\pgfqpoint{1.847cm}{1.202cm}}{\pgfqpoint{1.811cm}{1.202cm}}
\pgfpathcurveto{\pgfqpoint{1.775cm}{1.202cm}}{\pgfqpoint{1.74cm}{1.187cm}}{\pgfqpoint{1.714cm}{1.162cm}}
\pgfpathcurveto{\pgfqpoint{1.689cm}{1.136cm}}{\pgfqpoint{1.674cm}{1.101cm}}{\pgfqpoint{1.674cm}{1.065cm}}
\pgfpathcurveto{\pgfqpoint{1.674cm}{1.029cm}}{\pgfqpoint{1.689cm}{0.994cm}}{\pgfqpoint{1.714cm}{0.968cm}}
\pgfpathcurveto{\pgfqpoint{1.74cm}{0.942cm}}{\pgfqpoint{1.775cm}{0.928cm}}{\pgfqpoint{1.811cm}{0.928cm}}
\pgfpathcurveto{\pgfqpoint{1.847cm}{0.928cm}}{\pgfqpoint{1.882cm}{0.942cm}}{\pgfqpoint{1.907cm}{0.968cm}}
\pgfpathcurveto{\pgfqpoint{1.933cm}{0.994cm}}{\pgfqpoint{1.948cm}{1.029cm}}{\pgfqpoint{1.948cm}{1.065cm}}
\pgfusepath{fill}
\begin{pgfscope}
\pgfsetdash{}{0cm}
\pgfsetlinewidth{0.818mm}
\pgfsetmiterlimit{4.0}
\pgfpathmoveto{\pgfqpoint{1.383cm}{0.178cm}}
\pgfpathcurveto{\pgfqpoint{1.383cm}{0.214cm}}{\pgfqpoint{1.369cm}{0.249cm}}{\pgfqpoint{1.343cm}{0.275cm}}
\pgfpathcurveto{\pgfqpoint{1.317cm}{0.3cm}}{\pgfqpoint{1.283cm}{0.315cm}}{\pgfqpoint{1.246cm}{0.315cm}}
\pgfpathcurveto{\pgfqpoint{1.21cm}{0.315cm}}{\pgfqpoint{1.175cm}{0.3cm}}{\pgfqpoint{1.15cm}{0.275cm}}
\pgfpathcurveto{\pgfqpoint{1.124cm}{0.249cm}}{\pgfqpoint{1.11cm}{0.214cm}}{\pgfqpoint{1.11cm}{0.178cm}}
\pgfpathcurveto{\pgfqpoint{1.11cm}{0.141cm}}{\pgfqpoint{1.124cm}{0.107cm}}{\pgfqpoint{1.15cm}{0.081cm}}
\pgfpathcurveto{\pgfqpoint{1.175cm}{0.055cm}}{\pgfqpoint{1.21cm}{0.041cm}}{\pgfqpoint{1.246cm}{0.041cm}}
\pgfpathcurveto{\pgfqpoint{1.283cm}{0.041cm}}{\pgfqpoint{1.317cm}{0.055cm}}{\pgfqpoint{1.343cm}{0.081cm}}
\pgfpathcurveto{\pgfqpoint{1.369cm}{0.107cm}}{\pgfqpoint{1.383cm}{0.141cm}}{\pgfqpoint{1.383cm}{0.178cm}}
\pgfusepath{stroke}
\end{pgfscope}
\end{pgfscope}
\end{pgfscope}
\end{pgfscope}
\end{tikzpicture}}}\|_{\CC^{-2+\alpha}}\\
&\lesssim 2^{-(2-\alpha-\kappa)K/2}\|\phi+\psi\|_{L^\infty(\rho)},
\end{align*}
provided
\begin{equation}\label{eq:l4}
\|\UU_> X^{\!\resizebox{!}{.8em}{
\begin{tikzpicture}
\pgfpathmoveto{\pgfqpoint{0cm}{-0.035cm}}
\pgfpathlineto{\pgfqpoint{1.976cm}{-0.035cm}}
\pgfpathlineto{\pgfqpoint{1.976cm}{1.94cm}}
\pgfpathlineto{\pgfqpoint{0cm}{1.94cm}}
\pgfpathclose
\pgfusepath{clip}
\begin{pgfscope}
\begin{pgfscope}
\pgfpathmoveto{\pgfqpoint{0cm}{-0.035cm}}
\pgfpathlineto{\pgfqpoint{1.976cm}{-0.035cm}}
\pgfpathlineto{\pgfqpoint{1.976cm}{1.94cm}}
\pgfpathlineto{\pgfqpoint{0cm}{1.94cm}}
\pgfpathclose
\pgfusepath{clip}
\begin{pgfscope}
\begin{pgfscope}
\pgfsetdash{}{0cm}
\pgfsetlinewidth{0.818mm}
\pgfsetroundcap
\pgfsetroundjoin
\pgfsetmiterlimit{7.0}
\definecolor{eps2pgf_color}{gray}{0}\pgfsetstrokecolor{eps2pgf_color}\pgfsetfillcolor{eps2pgf_color}
\pgfpathmoveto{\pgfqpoint{0.117cm}{1.815cm}}
\pgfpathlineto{\pgfqpoint{0.682cm}{1.065cm}}
\pgfpathlineto{\pgfqpoint{1.246cm}{1.815cm}}
\pgfusepath{stroke}
\end{pgfscope}
\definecolor{eps2pgf_color}{gray}{0}\pgfsetstrokecolor{eps2pgf_color}\pgfsetfillcolor{eps2pgf_color}
\pgfpathmoveto{\pgfqpoint{0.273cm}{1.789cm}}
\pgfpathcurveto{\pgfqpoint{0.273cm}{1.825cm}}{\pgfqpoint{0.259cm}{1.86cm}}{\pgfqpoint{0.233cm}{1.886cm}}
\pgfpathcurveto{\pgfqpoint{0.207cm}{1.912cm}}{\pgfqpoint{0.173cm}{1.926cm}}{\pgfqpoint{0.137cm}{1.926cm}}
\pgfpathcurveto{\pgfqpoint{0.1cm}{1.926cm}}{\pgfqpoint{0.066cm}{1.912cm}}{\pgfqpoint{0.04cm}{1.886cm}}
\pgfpathcurveto{\pgfqpoint{0.014cm}{1.86cm}}{\pgfqpoint{0cm}{1.825cm}}{\pgfqpoint{0cm}{1.789cm}}
\pgfpathcurveto{\pgfqpoint{0cm}{1.753cm}}{\pgfqpoint{0.014cm}{1.718cm}}{\pgfqpoint{0.04cm}{1.692cm}}
\pgfpathcurveto{\pgfqpoint{0.066cm}{1.667cm}}{\pgfqpoint{0.1cm}{1.652cm}}{\pgfqpoint{0.137cm}{1.652cm}}
\pgfpathcurveto{\pgfqpoint{0.173cm}{1.652cm}}{\pgfqpoint{0.207cm}{1.667cm}}{\pgfqpoint{0.233cm}{1.692cm}}
\pgfpathcurveto{\pgfqpoint{0.259cm}{1.718cm}}{\pgfqpoint{0.273cm}{1.753cm}}{\pgfqpoint{0.273cm}{1.789cm}}
\pgfusepath{fill}
\begin{pgfscope}
\pgfsetdash{}{0cm}
\pgfsetlinewidth{0.818mm}
\pgfsetmiterlimit{7.0}
\pgfpathmoveto{\pgfqpoint{0.682cm}{1.065cm}}
\pgfpathlineto{\pgfqpoint{0.679cm}{1.812cm}}
\pgfusepath{stroke}
\end{pgfscope}
\pgfpathmoveto{\pgfqpoint{0.815cm}{1.793cm}}
\pgfpathcurveto{\pgfqpoint{0.815cm}{1.829cm}}{\pgfqpoint{0.801cm}{1.864cm}}{\pgfqpoint{0.775cm}{1.89cm}}
\pgfpathcurveto{\pgfqpoint{0.75cm}{1.915cm}}{\pgfqpoint{0.715cm}{1.93cm}}{\pgfqpoint{0.679cm}{1.93cm}}
\pgfpathcurveto{\pgfqpoint{0.643cm}{1.93cm}}{\pgfqpoint{0.608cm}{1.915cm}}{\pgfqpoint{0.582cm}{1.89cm}}
\pgfpathcurveto{\pgfqpoint{0.557cm}{1.864cm}}{\pgfqpoint{0.542cm}{1.829cm}}{\pgfqpoint{0.542cm}{1.793cm}}
\pgfpathcurveto{\pgfqpoint{0.542cm}{1.756cm}}{\pgfqpoint{0.557cm}{1.722cm}}{\pgfqpoint{0.582cm}{1.696cm}}
\pgfpathcurveto{\pgfqpoint{0.608cm}{1.67cm}}{\pgfqpoint{0.643cm}{1.656cm}}{\pgfqpoint{0.679cm}{1.656cm}}
\pgfpathcurveto{\pgfqpoint{0.715cm}{1.656cm}}{\pgfqpoint{0.75cm}{1.67cm}}{\pgfqpoint{0.775cm}{1.696cm}}
\pgfpathcurveto{\pgfqpoint{0.801cm}{1.722cm}}{\pgfqpoint{0.815cm}{1.756cm}}{\pgfqpoint{0.815cm}{1.793cm}}
\pgfusepath{fill}
\pgfpathmoveto{\pgfqpoint{1.345cm}{1.765cm}}
\pgfpathcurveto{\pgfqpoint{1.345cm}{1.801cm}}{\pgfqpoint{1.331cm}{1.836cm}}{\pgfqpoint{1.305cm}{1.862cm}}
\pgfpathcurveto{\pgfqpoint{1.28cm}{1.887cm}}{\pgfqpoint{1.245cm}{1.902cm}}{\pgfqpoint{1.209cm}{1.902cm}}
\pgfpathcurveto{\pgfqpoint{1.172cm}{1.902cm}}{\pgfqpoint{1.138cm}{1.887cm}}{\pgfqpoint{1.112cm}{1.862cm}}
\pgfpathcurveto{\pgfqpoint{1.087cm}{1.836cm}}{\pgfqpoint{1.072cm}{1.801cm}}{\pgfqpoint{1.072cm}{1.765cm}}
\pgfpathcurveto{\pgfqpoint{1.072cm}{1.728cm}}{\pgfqpoint{1.087cm}{1.694cm}}{\pgfqpoint{1.112cm}{1.668cm}}
\pgfpathcurveto{\pgfqpoint{1.138cm}{1.642cm}}{\pgfqpoint{1.172cm}{1.628cm}}{\pgfqpoint{1.209cm}{1.628cm}}
\pgfpathcurveto{\pgfqpoint{1.245cm}{1.628cm}}{\pgfqpoint{1.28cm}{1.642cm}}{\pgfqpoint{1.305cm}{1.668cm}}
\pgfpathcurveto{\pgfqpoint{1.331cm}{1.694cm}}{\pgfqpoint{1.345cm}{1.728cm}}{\pgfqpoint{1.345cm}{1.765cm}}
\pgfusepath{fill}
\begin{pgfscope}
\pgfsetdash{}{0cm}
\pgfsetlinewidth{0.818mm}
\pgfsetroundcap
\pgfsetroundjoin
\pgfsetmiterlimit{7.0}
\pgfpathmoveto{\pgfqpoint{0.682cm}{1.065cm}}
\pgfpathlineto{\pgfqpoint{1.246cm}{0.315cm}}
\pgfpathlineto{\pgfqpoint{1.811cm}{1.065cm}}
\pgfusepath{stroke}
\end{pgfscope}
\pgfpathmoveto{\pgfqpoint{1.948cm}{1.065cm}}
\pgfpathcurveto{\pgfqpoint{1.948cm}{1.101cm}}{\pgfqpoint{1.933cm}{1.136cm}}{\pgfqpoint{1.907cm}{1.162cm}}
\pgfpathcurveto{\pgfqpoint{1.882cm}{1.187cm}}{\pgfqpoint{1.847cm}{1.202cm}}{\pgfqpoint{1.811cm}{1.202cm}}
\pgfpathcurveto{\pgfqpoint{1.775cm}{1.202cm}}{\pgfqpoint{1.74cm}{1.187cm}}{\pgfqpoint{1.714cm}{1.162cm}}
\pgfpathcurveto{\pgfqpoint{1.689cm}{1.136cm}}{\pgfqpoint{1.674cm}{1.101cm}}{\pgfqpoint{1.674cm}{1.065cm}}
\pgfpathcurveto{\pgfqpoint{1.674cm}{1.029cm}}{\pgfqpoint{1.689cm}{0.994cm}}{\pgfqpoint{1.714cm}{0.968cm}}
\pgfpathcurveto{\pgfqpoint{1.74cm}{0.942cm}}{\pgfqpoint{1.775cm}{0.928cm}}{\pgfqpoint{1.811cm}{0.928cm}}
\pgfpathcurveto{\pgfqpoint{1.847cm}{0.928cm}}{\pgfqpoint{1.882cm}{0.942cm}}{\pgfqpoint{1.907cm}{0.968cm}}
\pgfpathcurveto{\pgfqpoint{1.933cm}{0.994cm}}{\pgfqpoint{1.948cm}{1.029cm}}{\pgfqpoint{1.948cm}{1.065cm}}
\pgfusepath{fill}
\begin{pgfscope}
\pgfsetdash{}{0cm}
\pgfsetlinewidth{0.818mm}
\pgfsetmiterlimit{4.0}
\pgfpathmoveto{\pgfqpoint{1.383cm}{0.178cm}}
\pgfpathcurveto{\pgfqpoint{1.383cm}{0.214cm}}{\pgfqpoint{1.369cm}{0.249cm}}{\pgfqpoint{1.343cm}{0.275cm}}
\pgfpathcurveto{\pgfqpoint{1.317cm}{0.3cm}}{\pgfqpoint{1.283cm}{0.315cm}}{\pgfqpoint{1.246cm}{0.315cm}}
\pgfpathcurveto{\pgfqpoint{1.21cm}{0.315cm}}{\pgfqpoint{1.175cm}{0.3cm}}{\pgfqpoint{1.15cm}{0.275cm}}
\pgfpathcurveto{\pgfqpoint{1.124cm}{0.249cm}}{\pgfqpoint{1.11cm}{0.214cm}}{\pgfqpoint{1.11cm}{0.178cm}}
\pgfpathcurveto{\pgfqpoint{1.11cm}{0.141cm}}{\pgfqpoint{1.124cm}{0.107cm}}{\pgfqpoint{1.15cm}{0.081cm}}
\pgfpathcurveto{\pgfqpoint{1.175cm}{0.055cm}}{\pgfqpoint{1.21cm}{0.041cm}}{\pgfqpoint{1.246cm}{0.041cm}}
\pgfpathcurveto{\pgfqpoint{1.283cm}{0.041cm}}{\pgfqpoint{1.317cm}{0.055cm}}{\pgfqpoint{1.343cm}{0.081cm}}
\pgfpathcurveto{\pgfqpoint{1.369cm}{0.107cm}}{\pgfqpoint{1.383cm}{0.141cm}}{\pgfqpoint{1.383cm}{0.178cm}}
\pgfusepath{stroke}
\end{pgfscope}
\end{pgfscope}
\end{pgfscope}
\end{pgfscope}
\end{tikzpicture}}}\|_{\CC^{-2+\alpha}}\lesssim 2^{-(2-\alpha-\kappa)K/2} \|X^{\!\resizebox{!}{.8em}{
\begin{tikzpicture}
\pgfpathmoveto{\pgfqpoint{0cm}{-0.035cm}}
\pgfpathlineto{\pgfqpoint{1.976cm}{-0.035cm}}
\pgfpathlineto{\pgfqpoint{1.976cm}{1.94cm}}
\pgfpathlineto{\pgfqpoint{0cm}{1.94cm}}
\pgfpathclose
\pgfusepath{clip}
\begin{pgfscope}
\begin{pgfscope}
\pgfpathmoveto{\pgfqpoint{0cm}{-0.035cm}}
\pgfpathlineto{\pgfqpoint{1.976cm}{-0.035cm}}
\pgfpathlineto{\pgfqpoint{1.976cm}{1.94cm}}
\pgfpathlineto{\pgfqpoint{0cm}{1.94cm}}
\pgfpathclose
\pgfusepath{clip}
\begin{pgfscope}
\begin{pgfscope}
\pgfsetdash{}{0cm}
\pgfsetlinewidth{0.818mm}
\pgfsetroundcap
\pgfsetroundjoin
\pgfsetmiterlimit{7.0}
\definecolor{eps2pgf_color}{gray}{0}\pgfsetstrokecolor{eps2pgf_color}\pgfsetfillcolor{eps2pgf_color}
\pgfpathmoveto{\pgfqpoint{0.117cm}{1.815cm}}
\pgfpathlineto{\pgfqpoint{0.682cm}{1.065cm}}
\pgfpathlineto{\pgfqpoint{1.246cm}{1.815cm}}
\pgfusepath{stroke}
\end{pgfscope}
\definecolor{eps2pgf_color}{gray}{0}\pgfsetstrokecolor{eps2pgf_color}\pgfsetfillcolor{eps2pgf_color}
\pgfpathmoveto{\pgfqpoint{0.273cm}{1.789cm}}
\pgfpathcurveto{\pgfqpoint{0.273cm}{1.825cm}}{\pgfqpoint{0.259cm}{1.86cm}}{\pgfqpoint{0.233cm}{1.886cm}}
\pgfpathcurveto{\pgfqpoint{0.207cm}{1.912cm}}{\pgfqpoint{0.173cm}{1.926cm}}{\pgfqpoint{0.137cm}{1.926cm}}
\pgfpathcurveto{\pgfqpoint{0.1cm}{1.926cm}}{\pgfqpoint{0.066cm}{1.912cm}}{\pgfqpoint{0.04cm}{1.886cm}}
\pgfpathcurveto{\pgfqpoint{0.014cm}{1.86cm}}{\pgfqpoint{0cm}{1.825cm}}{\pgfqpoint{0cm}{1.789cm}}
\pgfpathcurveto{\pgfqpoint{0cm}{1.753cm}}{\pgfqpoint{0.014cm}{1.718cm}}{\pgfqpoint{0.04cm}{1.692cm}}
\pgfpathcurveto{\pgfqpoint{0.066cm}{1.667cm}}{\pgfqpoint{0.1cm}{1.652cm}}{\pgfqpoint{0.137cm}{1.652cm}}
\pgfpathcurveto{\pgfqpoint{0.173cm}{1.652cm}}{\pgfqpoint{0.207cm}{1.667cm}}{\pgfqpoint{0.233cm}{1.692cm}}
\pgfpathcurveto{\pgfqpoint{0.259cm}{1.718cm}}{\pgfqpoint{0.273cm}{1.753cm}}{\pgfqpoint{0.273cm}{1.789cm}}
\pgfusepath{fill}
\begin{pgfscope}
\pgfsetdash{}{0cm}
\pgfsetlinewidth{0.818mm}
\pgfsetmiterlimit{7.0}
\pgfpathmoveto{\pgfqpoint{0.682cm}{1.065cm}}
\pgfpathlineto{\pgfqpoint{0.679cm}{1.812cm}}
\pgfusepath{stroke}
\end{pgfscope}
\pgfpathmoveto{\pgfqpoint{0.815cm}{1.793cm}}
\pgfpathcurveto{\pgfqpoint{0.815cm}{1.829cm}}{\pgfqpoint{0.801cm}{1.864cm}}{\pgfqpoint{0.775cm}{1.89cm}}
\pgfpathcurveto{\pgfqpoint{0.75cm}{1.915cm}}{\pgfqpoint{0.715cm}{1.93cm}}{\pgfqpoint{0.679cm}{1.93cm}}
\pgfpathcurveto{\pgfqpoint{0.643cm}{1.93cm}}{\pgfqpoint{0.608cm}{1.915cm}}{\pgfqpoint{0.582cm}{1.89cm}}
\pgfpathcurveto{\pgfqpoint{0.557cm}{1.864cm}}{\pgfqpoint{0.542cm}{1.829cm}}{\pgfqpoint{0.542cm}{1.793cm}}
\pgfpathcurveto{\pgfqpoint{0.542cm}{1.756cm}}{\pgfqpoint{0.557cm}{1.722cm}}{\pgfqpoint{0.582cm}{1.696cm}}
\pgfpathcurveto{\pgfqpoint{0.608cm}{1.67cm}}{\pgfqpoint{0.643cm}{1.656cm}}{\pgfqpoint{0.679cm}{1.656cm}}
\pgfpathcurveto{\pgfqpoint{0.715cm}{1.656cm}}{\pgfqpoint{0.75cm}{1.67cm}}{\pgfqpoint{0.775cm}{1.696cm}}
\pgfpathcurveto{\pgfqpoint{0.801cm}{1.722cm}}{\pgfqpoint{0.815cm}{1.756cm}}{\pgfqpoint{0.815cm}{1.793cm}}
\pgfusepath{fill}
\pgfpathmoveto{\pgfqpoint{1.345cm}{1.765cm}}
\pgfpathcurveto{\pgfqpoint{1.345cm}{1.801cm}}{\pgfqpoint{1.331cm}{1.836cm}}{\pgfqpoint{1.305cm}{1.862cm}}
\pgfpathcurveto{\pgfqpoint{1.28cm}{1.887cm}}{\pgfqpoint{1.245cm}{1.902cm}}{\pgfqpoint{1.209cm}{1.902cm}}
\pgfpathcurveto{\pgfqpoint{1.172cm}{1.902cm}}{\pgfqpoint{1.138cm}{1.887cm}}{\pgfqpoint{1.112cm}{1.862cm}}
\pgfpathcurveto{\pgfqpoint{1.087cm}{1.836cm}}{\pgfqpoint{1.072cm}{1.801cm}}{\pgfqpoint{1.072cm}{1.765cm}}
\pgfpathcurveto{\pgfqpoint{1.072cm}{1.728cm}}{\pgfqpoint{1.087cm}{1.694cm}}{\pgfqpoint{1.112cm}{1.668cm}}
\pgfpathcurveto{\pgfqpoint{1.138cm}{1.642cm}}{\pgfqpoint{1.172cm}{1.628cm}}{\pgfqpoint{1.209cm}{1.628cm}}
\pgfpathcurveto{\pgfqpoint{1.245cm}{1.628cm}}{\pgfqpoint{1.28cm}{1.642cm}}{\pgfqpoint{1.305cm}{1.668cm}}
\pgfpathcurveto{\pgfqpoint{1.331cm}{1.694cm}}{\pgfqpoint{1.345cm}{1.728cm}}{\pgfqpoint{1.345cm}{1.765cm}}
\pgfusepath{fill}
\begin{pgfscope}
\pgfsetdash{}{0cm}
\pgfsetlinewidth{0.818mm}
\pgfsetroundcap
\pgfsetroundjoin
\pgfsetmiterlimit{7.0}
\pgfpathmoveto{\pgfqpoint{0.682cm}{1.065cm}}
\pgfpathlineto{\pgfqpoint{1.246cm}{0.315cm}}
\pgfpathlineto{\pgfqpoint{1.811cm}{1.065cm}}
\pgfusepath{stroke}
\end{pgfscope}
\pgfpathmoveto{\pgfqpoint{1.948cm}{1.065cm}}
\pgfpathcurveto{\pgfqpoint{1.948cm}{1.101cm}}{\pgfqpoint{1.933cm}{1.136cm}}{\pgfqpoint{1.907cm}{1.162cm}}
\pgfpathcurveto{\pgfqpoint{1.882cm}{1.187cm}}{\pgfqpoint{1.847cm}{1.202cm}}{\pgfqpoint{1.811cm}{1.202cm}}
\pgfpathcurveto{\pgfqpoint{1.775cm}{1.202cm}}{\pgfqpoint{1.74cm}{1.187cm}}{\pgfqpoint{1.714cm}{1.162cm}}
\pgfpathcurveto{\pgfqpoint{1.689cm}{1.136cm}}{\pgfqpoint{1.674cm}{1.101cm}}{\pgfqpoint{1.674cm}{1.065cm}}
\pgfpathcurveto{\pgfqpoint{1.674cm}{1.029cm}}{\pgfqpoint{1.689cm}{0.994cm}}{\pgfqpoint{1.714cm}{0.968cm}}
\pgfpathcurveto{\pgfqpoint{1.74cm}{0.942cm}}{\pgfqpoint{1.775cm}{0.928cm}}{\pgfqpoint{1.811cm}{0.928cm}}
\pgfpathcurveto{\pgfqpoint{1.847cm}{0.928cm}}{\pgfqpoint{1.882cm}{0.942cm}}{\pgfqpoint{1.907cm}{0.968cm}}
\pgfpathcurveto{\pgfqpoint{1.933cm}{0.994cm}}{\pgfqpoint{1.948cm}{1.029cm}}{\pgfqpoint{1.948cm}{1.065cm}}
\pgfusepath{fill}
\begin{pgfscope}
\pgfsetdash{}{0cm}
\pgfsetlinewidth{0.818mm}
\pgfsetmiterlimit{4.0}
\pgfpathmoveto{\pgfqpoint{1.383cm}{0.178cm}}
\pgfpathcurveto{\pgfqpoint{1.383cm}{0.214cm}}{\pgfqpoint{1.369cm}{0.249cm}}{\pgfqpoint{1.343cm}{0.275cm}}
\pgfpathcurveto{\pgfqpoint{1.317cm}{0.3cm}}{\pgfqpoint{1.283cm}{0.315cm}}{\pgfqpoint{1.246cm}{0.315cm}}
\pgfpathcurveto{\pgfqpoint{1.21cm}{0.315cm}}{\pgfqpoint{1.175cm}{0.3cm}}{\pgfqpoint{1.15cm}{0.275cm}}
\pgfpathcurveto{\pgfqpoint{1.124cm}{0.249cm}}{\pgfqpoint{1.11cm}{0.214cm}}{\pgfqpoint{1.11cm}{0.178cm}}
\pgfpathcurveto{\pgfqpoint{1.11cm}{0.141cm}}{\pgfqpoint{1.124cm}{0.107cm}}{\pgfqpoint{1.15cm}{0.081cm}}
\pgfpathcurveto{\pgfqpoint{1.175cm}{0.055cm}}{\pgfqpoint{1.21cm}{0.041cm}}{\pgfqpoint{1.246cm}{0.041cm}}
\pgfpathcurveto{\pgfqpoint{1.283cm}{0.041cm}}{\pgfqpoint{1.317cm}{0.055cm}}{\pgfqpoint{1.343cm}{0.081cm}}
\pgfpathcurveto{\pgfqpoint{1.369cm}{0.107cm}}{\pgfqpoint{1.383cm}{0.141cm}}{\pgfqpoint{1.383cm}{0.178cm}}
\pgfusepath{stroke}
\end{pgfscope}
\end{pgfscope}
\end{pgfscope}
\end{pgfscope}
\end{tikzpicture}}} \|_{\CC^{-\kappa}(\rho^\sigma)}.
\end{equation}
and finally
\begin{align*}
\|3\UU_> X\succ(\phi+\psi)^2\|_{\CC^{-2+\alpha}(\rho)}
&\lesssim \|\phi+\psi\|^2_{L^\infty(\rho)}\|\UU_>X\|_{\CC^{-2+\alpha}(\rho^{-1})}\\
&\lesssim 2^{-(\frac{3}{2}-\alpha-\kappa)\frac{4}{3}K}\|\phi+\psi\|_{L^\infty(\rho)}^2,
\end{align*}
provided
\begin{equation}\label{eq:l5}
\|\UU_> X\|_{\CC^{-2+\alpha}(\rho^{-1})}\lesssim 2^{-(\frac{3}{2}-\alpha-\kappa)\frac{4}{3}K} \|X\|_{\CC^{-\frac{1}{2}-\kappa}(\rho^\sigma)}.
\end{equation}

In view of Lemma \ref{lem:local}, once the weight $\rho$ is fixed, the value of $L$ completely determines how the associated localizers $\UU_>$ and $\UU_{\leq}$ are defined. The above considerations  and in particular \eqref{eq:l1}, \eqref{eq:l2}, \eqref{eq:l3}, \eqref{eq:l4}, \eqref{eq:l5} lead us  the  values of $L$ for various objects in our expansion summarized  (in chronological order) in Table \ref{t:loc}.

  \begin{table}
    \begin{center}
  \begin{tabular}{| r | c | c | c | c|c|}
  \hline
  Object &  $ \llbracket X^2 \rrbracket$ & $ X^{\!\resizebox{!}{.8em}{
\begin{tikzpicture}
\pgfpathmoveto{\pgfqpoint{0cm}{-0.035cm}}
\pgfpathlineto{\pgfqpoint{1.976cm}{-0.035cm}}
\pgfpathlineto{\pgfqpoint{1.976cm}{1.94cm}}
\pgfpathlineto{\pgfqpoint{0cm}{1.94cm}}
\pgfpathclose
\pgfusepath{clip}
\begin{pgfscope}
\begin{pgfscope}
\pgfpathmoveto{\pgfqpoint{0cm}{-0.035cm}}
\pgfpathlineto{\pgfqpoint{1.976cm}{-0.035cm}}
\pgfpathlineto{\pgfqpoint{1.976cm}{1.94cm}}
\pgfpathlineto{\pgfqpoint{0cm}{1.94cm}}
\pgfpathclose
\pgfusepath{clip}
\begin{pgfscope}
\begin{pgfscope}
\pgfsetdash{}{0cm}
\pgfsetlinewidth{0.818mm}
\pgfsetroundcap
\pgfsetroundjoin
\pgfsetmiterlimit{7.0}
\definecolor{eps2pgf_color}{gray}{0}\pgfsetstrokecolor{eps2pgf_color}\pgfsetfillcolor{eps2pgf_color}
\pgfpathmoveto{\pgfqpoint{0.117cm}{1.815cm}}
\pgfpathlineto{\pgfqpoint{0.682cm}{1.065cm}}
\pgfpathlineto{\pgfqpoint{1.246cm}{1.815cm}}
\pgfusepath{stroke}
\end{pgfscope}
\definecolor{eps2pgf_color}{gray}{0}\pgfsetstrokecolor{eps2pgf_color}\pgfsetfillcolor{eps2pgf_color}
\pgfpathmoveto{\pgfqpoint{0.273cm}{1.789cm}}
\pgfpathcurveto{\pgfqpoint{0.273cm}{1.825cm}}{\pgfqpoint{0.259cm}{1.86cm}}{\pgfqpoint{0.233cm}{1.886cm}}
\pgfpathcurveto{\pgfqpoint{0.207cm}{1.912cm}}{\pgfqpoint{0.173cm}{1.926cm}}{\pgfqpoint{0.137cm}{1.926cm}}
\pgfpathcurveto{\pgfqpoint{0.1cm}{1.926cm}}{\pgfqpoint{0.066cm}{1.912cm}}{\pgfqpoint{0.04cm}{1.886cm}}
\pgfpathcurveto{\pgfqpoint{0.014cm}{1.86cm}}{\pgfqpoint{0cm}{1.825cm}}{\pgfqpoint{0cm}{1.789cm}}
\pgfpathcurveto{\pgfqpoint{0cm}{1.753cm}}{\pgfqpoint{0.014cm}{1.718cm}}{\pgfqpoint{0.04cm}{1.692cm}}
\pgfpathcurveto{\pgfqpoint{0.066cm}{1.667cm}}{\pgfqpoint{0.1cm}{1.652cm}}{\pgfqpoint{0.137cm}{1.652cm}}
\pgfpathcurveto{\pgfqpoint{0.173cm}{1.652cm}}{\pgfqpoint{0.207cm}{1.667cm}}{\pgfqpoint{0.233cm}{1.692cm}}
\pgfpathcurveto{\pgfqpoint{0.259cm}{1.718cm}}{\pgfqpoint{0.273cm}{1.753cm}}{\pgfqpoint{0.273cm}{1.789cm}}
\pgfusepath{fill}
\pgfpathmoveto{\pgfqpoint{1.345cm}{1.765cm}}
\pgfpathcurveto{\pgfqpoint{1.345cm}{1.801cm}}{\pgfqpoint{1.331cm}{1.836cm}}{\pgfqpoint{1.305cm}{1.862cm}}
\pgfpathcurveto{\pgfqpoint{1.28cm}{1.887cm}}{\pgfqpoint{1.245cm}{1.902cm}}{\pgfqpoint{1.209cm}{1.902cm}}
\pgfpathcurveto{\pgfqpoint{1.172cm}{1.902cm}}{\pgfqpoint{1.138cm}{1.887cm}}{\pgfqpoint{1.112cm}{1.862cm}}
\pgfpathcurveto{\pgfqpoint{1.087cm}{1.836cm}}{\pgfqpoint{1.072cm}{1.801cm}}{\pgfqpoint{1.072cm}{1.765cm}}
\pgfpathcurveto{\pgfqpoint{1.072cm}{1.728cm}}{\pgfqpoint{1.087cm}{1.694cm}}{\pgfqpoint{1.112cm}{1.668cm}}
\pgfpathcurveto{\pgfqpoint{1.138cm}{1.642cm}}{\pgfqpoint{1.172cm}{1.628cm}}{\pgfqpoint{1.209cm}{1.628cm}}
\pgfpathcurveto{\pgfqpoint{1.245cm}{1.628cm}}{\pgfqpoint{1.28cm}{1.642cm}}{\pgfqpoint{1.305cm}{1.668cm}}
\pgfpathcurveto{\pgfqpoint{1.331cm}{1.694cm}}{\pgfqpoint{1.345cm}{1.728cm}}{\pgfqpoint{1.345cm}{1.765cm}}
\pgfusepath{fill}
\begin{pgfscope}
\pgfsetdash{}{0cm}
\pgfsetlinewidth{0.818mm}
\pgfsetroundcap
\pgfsetroundjoin
\pgfsetmiterlimit{7.0}
\pgfpathmoveto{\pgfqpoint{0.682cm}{1.065cm}}
\pgfpathlineto{\pgfqpoint{1.246cm}{0.315cm}}
\pgfpathlineto{\pgfqpoint{1.811cm}{1.065cm}}
\pgfusepath{stroke}
\end{pgfscope}
\pgfpathmoveto{\pgfqpoint{1.948cm}{1.065cm}}
\pgfpathcurveto{\pgfqpoint{1.948cm}{1.101cm}}{\pgfqpoint{1.933cm}{1.136cm}}{\pgfqpoint{1.907cm}{1.162cm}}
\pgfpathcurveto{\pgfqpoint{1.882cm}{1.187cm}}{\pgfqpoint{1.847cm}{1.202cm}}{\pgfqpoint{1.811cm}{1.202cm}}
\pgfpathcurveto{\pgfqpoint{1.775cm}{1.202cm}}{\pgfqpoint{1.74cm}{1.187cm}}{\pgfqpoint{1.714cm}{1.162cm}}
\pgfpathcurveto{\pgfqpoint{1.689cm}{1.136cm}}{\pgfqpoint{1.674cm}{1.101cm}}{\pgfqpoint{1.674cm}{1.065cm}}
\pgfpathcurveto{\pgfqpoint{1.674cm}{1.029cm}}{\pgfqpoint{1.689cm}{0.994cm}}{\pgfqpoint{1.714cm}{0.968cm}}
\pgfpathcurveto{\pgfqpoint{1.74cm}{0.942cm}}{\pgfqpoint{1.775cm}{0.928cm}}{\pgfqpoint{1.811cm}{0.928cm}}
\pgfpathcurveto{\pgfqpoint{1.847cm}{0.928cm}}{\pgfqpoint{1.882cm}{0.942cm}}{\pgfqpoint{1.907cm}{0.968cm}}
\pgfpathcurveto{\pgfqpoint{1.933cm}{0.994cm}}{\pgfqpoint{1.948cm}{1.029cm}}{\pgfqpoint{1.948cm}{1.065cm}}
\pgfusepath{fill}
\begin{pgfscope}
\pgfsetdash{}{0cm}
\pgfsetlinewidth{0.818mm}
\pgfsetmiterlimit{7.0}
\pgfpathmoveto{\pgfqpoint{1.246cm}{0.315cm}}
\pgfpathlineto{\pgfqpoint{1.244cm}{1.061cm}}
\pgfusepath{stroke}
\end{pgfscope}
\pgfpathmoveto{\pgfqpoint{1.38cm}{1.065cm}}
\pgfpathcurveto{\pgfqpoint{1.38cm}{1.101cm}}{\pgfqpoint{1.366cm}{1.136cm}}{\pgfqpoint{1.34cm}{1.162cm}}
\pgfpathcurveto{\pgfqpoint{1.315cm}{1.187cm}}{\pgfqpoint{1.28cm}{1.202cm}}{\pgfqpoint{1.244cm}{1.202cm}}
\pgfpathcurveto{\pgfqpoint{1.207cm}{1.202cm}}{\pgfqpoint{1.173cm}{1.187cm}}{\pgfqpoint{1.147cm}{1.162cm}}
\pgfpathcurveto{\pgfqpoint{1.121cm}{1.136cm}}{\pgfqpoint{1.107cm}{1.101cm}}{\pgfqpoint{1.107cm}{1.065cm}}
\pgfpathcurveto{\pgfqpoint{1.107cm}{1.029cm}}{\pgfqpoint{1.121cm}{0.994cm}}{\pgfqpoint{1.147cm}{0.968cm}}
\pgfpathcurveto{\pgfqpoint{1.173cm}{0.942cm}}{\pgfqpoint{1.207cm}{0.928cm}}{\pgfqpoint{1.244cm}{0.928cm}}
\pgfpathcurveto{\pgfqpoint{1.28cm}{0.928cm}}{\pgfqpoint{1.315cm}{0.942cm}}{\pgfqpoint{1.34cm}{0.968cm}}
\pgfpathcurveto{\pgfqpoint{1.366cm}{0.994cm}}{\pgfqpoint{1.38cm}{1.029cm}}{\pgfqpoint{1.38cm}{1.065cm}}
\pgfusepath{fill}
\begin{pgfscope}
\pgfsetdash{}{0cm}
\pgfsetlinewidth{0.818mm}
\pgfsetmiterlimit{4.0}
\pgfpathmoveto{\pgfqpoint{1.383cm}{0.178cm}}
\pgfpathcurveto{\pgfqpoint{1.383cm}{0.214cm}}{\pgfqpoint{1.369cm}{0.249cm}}{\pgfqpoint{1.343cm}{0.275cm}}
\pgfpathcurveto{\pgfqpoint{1.317cm}{0.3cm}}{\pgfqpoint{1.283cm}{0.315cm}}{\pgfqpoint{1.246cm}{0.315cm}}
\pgfpathcurveto{\pgfqpoint{1.21cm}{0.315cm}}{\pgfqpoint{1.175cm}{0.3cm}}{\pgfqpoint{1.15cm}{0.275cm}}
\pgfpathcurveto{\pgfqpoint{1.124cm}{0.249cm}}{\pgfqpoint{1.11cm}{0.214cm}}{\pgfqpoint{1.11cm}{0.178cm}}
\pgfpathcurveto{\pgfqpoint{1.11cm}{0.141cm}}{\pgfqpoint{1.124cm}{0.107cm}}{\pgfqpoint{1.15cm}{0.081cm}}
\pgfpathcurveto{\pgfqpoint{1.175cm}{0.055cm}}{\pgfqpoint{1.21cm}{0.041cm}}{\pgfqpoint{1.246cm}{0.041cm}}
\pgfpathcurveto{\pgfqpoint{1.283cm}{0.041cm}}{\pgfqpoint{1.317cm}{0.055cm}}{\pgfqpoint{1.343cm}{0.081cm}}
\pgfpathcurveto{\pgfqpoint{1.369cm}{0.107cm}}{\pgfqpoint{1.383cm}{0.141cm}}{\pgfqpoint{1.383cm}{0.178cm}}
\pgfusepath{stroke}
\end{pgfscope}
\end{pgfscope}
\end{pgfscope}
\end{pgfscope}
\end{tikzpicture}}}$  & $X$ &  $X^{\!\resizebox{!}{.8em}{
\begin{tikzpicture}
\pgfpathmoveto{\pgfqpoint{0cm}{-0.035cm}}
\pgfpathlineto{\pgfqpoint{1.976cm}{-0.035cm}}
\pgfpathlineto{\pgfqpoint{1.976cm}{1.94cm}}
\pgfpathlineto{\pgfqpoint{0cm}{1.94cm}}
\pgfpathclose
\pgfusepath{clip}
\begin{pgfscope}
\begin{pgfscope}
\pgfpathmoveto{\pgfqpoint{0cm}{-0.035cm}}
\pgfpathlineto{\pgfqpoint{1.976cm}{-0.035cm}}
\pgfpathlineto{\pgfqpoint{1.976cm}{1.94cm}}
\pgfpathlineto{\pgfqpoint{0cm}{1.94cm}}
\pgfpathclose
\pgfusepath{clip}
\begin{pgfscope}
\begin{pgfscope}
\pgfsetdash{}{0cm}
\pgfsetlinewidth{0.818mm}
\pgfsetroundcap
\pgfsetroundjoin
\pgfsetmiterlimit{7.0}
\definecolor{eps2pgf_color}{gray}{0}\pgfsetstrokecolor{eps2pgf_color}\pgfsetfillcolor{eps2pgf_color}
\pgfpathmoveto{\pgfqpoint{0.117cm}{1.815cm}}
\pgfpathlineto{\pgfqpoint{0.682cm}{1.065cm}}
\pgfpathlineto{\pgfqpoint{1.246cm}{1.815cm}}
\pgfusepath{stroke}
\end{pgfscope}
\definecolor{eps2pgf_color}{gray}{0}\pgfsetstrokecolor{eps2pgf_color}\pgfsetfillcolor{eps2pgf_color}
\pgfpathmoveto{\pgfqpoint{0.273cm}{1.789cm}}
\pgfpathcurveto{\pgfqpoint{0.273cm}{1.825cm}}{\pgfqpoint{0.259cm}{1.86cm}}{\pgfqpoint{0.233cm}{1.886cm}}
\pgfpathcurveto{\pgfqpoint{0.207cm}{1.912cm}}{\pgfqpoint{0.173cm}{1.926cm}}{\pgfqpoint{0.137cm}{1.926cm}}
\pgfpathcurveto{\pgfqpoint{0.1cm}{1.926cm}}{\pgfqpoint{0.066cm}{1.912cm}}{\pgfqpoint{0.04cm}{1.886cm}}
\pgfpathcurveto{\pgfqpoint{0.014cm}{1.86cm}}{\pgfqpoint{0cm}{1.825cm}}{\pgfqpoint{0cm}{1.789cm}}
\pgfpathcurveto{\pgfqpoint{0cm}{1.753cm}}{\pgfqpoint{0.014cm}{1.718cm}}{\pgfqpoint{0.04cm}{1.692cm}}
\pgfpathcurveto{\pgfqpoint{0.066cm}{1.667cm}}{\pgfqpoint{0.1cm}{1.652cm}}{\pgfqpoint{0.137cm}{1.652cm}}
\pgfpathcurveto{\pgfqpoint{0.173cm}{1.652cm}}{\pgfqpoint{0.207cm}{1.667cm}}{\pgfqpoint{0.233cm}{1.692cm}}
\pgfpathcurveto{\pgfqpoint{0.259cm}{1.718cm}}{\pgfqpoint{0.273cm}{1.753cm}}{\pgfqpoint{0.273cm}{1.789cm}}
\pgfusepath{fill}
\begin{pgfscope}
\pgfsetdash{}{0cm}
\pgfsetlinewidth{0.818mm}
\pgfsetmiterlimit{7.0}
\pgfpathmoveto{\pgfqpoint{0.682cm}{1.065cm}}
\pgfpathlineto{\pgfqpoint{0.679cm}{1.812cm}}
\pgfusepath{stroke}
\end{pgfscope}
\pgfpathmoveto{\pgfqpoint{0.815cm}{1.793cm}}
\pgfpathcurveto{\pgfqpoint{0.815cm}{1.829cm}}{\pgfqpoint{0.801cm}{1.864cm}}{\pgfqpoint{0.775cm}{1.89cm}}
\pgfpathcurveto{\pgfqpoint{0.75cm}{1.915cm}}{\pgfqpoint{0.715cm}{1.93cm}}{\pgfqpoint{0.679cm}{1.93cm}}
\pgfpathcurveto{\pgfqpoint{0.643cm}{1.93cm}}{\pgfqpoint{0.608cm}{1.915cm}}{\pgfqpoint{0.582cm}{1.89cm}}
\pgfpathcurveto{\pgfqpoint{0.557cm}{1.864cm}}{\pgfqpoint{0.542cm}{1.829cm}}{\pgfqpoint{0.542cm}{1.793cm}}
\pgfpathcurveto{\pgfqpoint{0.542cm}{1.756cm}}{\pgfqpoint{0.557cm}{1.722cm}}{\pgfqpoint{0.582cm}{1.696cm}}
\pgfpathcurveto{\pgfqpoint{0.608cm}{1.67cm}}{\pgfqpoint{0.643cm}{1.656cm}}{\pgfqpoint{0.679cm}{1.656cm}}
\pgfpathcurveto{\pgfqpoint{0.715cm}{1.656cm}}{\pgfqpoint{0.75cm}{1.67cm}}{\pgfqpoint{0.775cm}{1.696cm}}
\pgfpathcurveto{\pgfqpoint{0.801cm}{1.722cm}}{\pgfqpoint{0.815cm}{1.756cm}}{\pgfqpoint{0.815cm}{1.793cm}}
\pgfusepath{fill}
\pgfpathmoveto{\pgfqpoint{1.345cm}{1.765cm}}
\pgfpathcurveto{\pgfqpoint{1.345cm}{1.801cm}}{\pgfqpoint{1.331cm}{1.836cm}}{\pgfqpoint{1.305cm}{1.862cm}}
\pgfpathcurveto{\pgfqpoint{1.28cm}{1.887cm}}{\pgfqpoint{1.245cm}{1.902cm}}{\pgfqpoint{1.209cm}{1.902cm}}
\pgfpathcurveto{\pgfqpoint{1.172cm}{1.902cm}}{\pgfqpoint{1.138cm}{1.887cm}}{\pgfqpoint{1.112cm}{1.862cm}}
\pgfpathcurveto{\pgfqpoint{1.087cm}{1.836cm}}{\pgfqpoint{1.072cm}{1.801cm}}{\pgfqpoint{1.072cm}{1.765cm}}
\pgfpathcurveto{\pgfqpoint{1.072cm}{1.728cm}}{\pgfqpoint{1.087cm}{1.694cm}}{\pgfqpoint{1.112cm}{1.668cm}}
\pgfpathcurveto{\pgfqpoint{1.138cm}{1.642cm}}{\pgfqpoint{1.172cm}{1.628cm}}{\pgfqpoint{1.209cm}{1.628cm}}
\pgfpathcurveto{\pgfqpoint{1.245cm}{1.628cm}}{\pgfqpoint{1.28cm}{1.642cm}}{\pgfqpoint{1.305cm}{1.668cm}}
\pgfpathcurveto{\pgfqpoint{1.331cm}{1.694cm}}{\pgfqpoint{1.345cm}{1.728cm}}{\pgfqpoint{1.345cm}{1.765cm}}
\pgfusepath{fill}
\begin{pgfscope}
\pgfsetdash{}{0cm}
\pgfsetlinewidth{0.818mm}
\pgfsetroundcap
\pgfsetroundjoin
\pgfsetmiterlimit{7.0}
\pgfpathmoveto{\pgfqpoint{0.682cm}{1.065cm}}
\pgfpathlineto{\pgfqpoint{1.246cm}{0.315cm}}
\pgfpathlineto{\pgfqpoint{1.811cm}{1.065cm}}
\pgfusepath{stroke}
\end{pgfscope}
\pgfpathmoveto{\pgfqpoint{1.948cm}{1.065cm}}
\pgfpathcurveto{\pgfqpoint{1.948cm}{1.101cm}}{\pgfqpoint{1.933cm}{1.136cm}}{\pgfqpoint{1.907cm}{1.162cm}}
\pgfpathcurveto{\pgfqpoint{1.882cm}{1.187cm}}{\pgfqpoint{1.847cm}{1.202cm}}{\pgfqpoint{1.811cm}{1.202cm}}
\pgfpathcurveto{\pgfqpoint{1.775cm}{1.202cm}}{\pgfqpoint{1.74cm}{1.187cm}}{\pgfqpoint{1.714cm}{1.162cm}}
\pgfpathcurveto{\pgfqpoint{1.689cm}{1.136cm}}{\pgfqpoint{1.674cm}{1.101cm}}{\pgfqpoint{1.674cm}{1.065cm}}
\pgfpathcurveto{\pgfqpoint{1.674cm}{1.029cm}}{\pgfqpoint{1.689cm}{0.994cm}}{\pgfqpoint{1.714cm}{0.968cm}}
\pgfpathcurveto{\pgfqpoint{1.74cm}{0.942cm}}{\pgfqpoint{1.775cm}{0.928cm}}{\pgfqpoint{1.811cm}{0.928cm}}
\pgfpathcurveto{\pgfqpoint{1.847cm}{0.928cm}}{\pgfqpoint{1.882cm}{0.942cm}}{\pgfqpoint{1.907cm}{0.968cm}}
\pgfpathcurveto{\pgfqpoint{1.933cm}{0.994cm}}{\pgfqpoint{1.948cm}{1.029cm}}{\pgfqpoint{1.948cm}{1.065cm}}
\pgfusepath{fill}
\begin{pgfscope}
\pgfsetdash{}{0cm}
\pgfsetlinewidth{0.818mm}
\pgfsetmiterlimit{4.0}
\pgfpathmoveto{\pgfqpoint{1.383cm}{0.178cm}}
\pgfpathcurveto{\pgfqpoint{1.383cm}{0.214cm}}{\pgfqpoint{1.369cm}{0.249cm}}{\pgfqpoint{1.343cm}{0.275cm}}
\pgfpathcurveto{\pgfqpoint{1.317cm}{0.3cm}}{\pgfqpoint{1.283cm}{0.315cm}}{\pgfqpoint{1.246cm}{0.315cm}}
\pgfpathcurveto{\pgfqpoint{1.21cm}{0.315cm}}{\pgfqpoint{1.175cm}{0.3cm}}{\pgfqpoint{1.15cm}{0.275cm}}
\pgfpathcurveto{\pgfqpoint{1.124cm}{0.249cm}}{\pgfqpoint{1.11cm}{0.214cm}}{\pgfqpoint{1.11cm}{0.178cm}}
\pgfpathcurveto{\pgfqpoint{1.11cm}{0.141cm}}{\pgfqpoint{1.124cm}{0.107cm}}{\pgfqpoint{1.15cm}{0.081cm}}
\pgfpathcurveto{\pgfqpoint{1.175cm}{0.055cm}}{\pgfqpoint{1.21cm}{0.041cm}}{\pgfqpoint{1.246cm}{0.041cm}}
\pgfpathcurveto{\pgfqpoint{1.283cm}{0.041cm}}{\pgfqpoint{1.317cm}{0.055cm}}{\pgfqpoint{1.343cm}{0.081cm}}
\pgfpathcurveto{\pgfqpoint{1.369cm}{0.107cm}}{\pgfqpoint{1.383cm}{0.141cm}}{\pgfqpoint{1.383cm}{0.178cm}}
\pgfusepath{stroke}
\end{pgfscope}
\end{pgfscope}
\end{pgfscope}
\end{pgfscope}
\end{tikzpicture}}}$ & $ X$\\ \hline 
  $L$ & $K$ & $\frac{K}{2}$ & $\frac{2K}{3} $ & $\frac{K}{2}$ & $\frac{4K}{3}$  \\
  \hline
  \end{tabular}
  \caption{The value of parameter $L$ used to construct $\UU_>$, $\UU_\leq$.}
  \label{t:loc}
    \end{center}
    \end{table}

Collecting all the above estimates and using the Schauder estimates we  deduce that 
\begin{align*}
&\|\phi\|_{L^\infty(\rho)}\lesssim  \|\phi\|_{\CC^\alpha(\rho)} \\
&\lesssim\|\Phi\|_{\CC^{-2+\alpha}(\rho)}\lesssim 1+2^{-(1-\alpha-\kappa)K}\|\phi+\psi\|_{L^\infty(\rho)}+2^{-(\frac{3}{2}-\alpha-\kappa)\frac{4}{3}K}\|\phi+\psi\|^2_{L^\infty(\rho)}
\end{align*}
\rmbb{\rmbb{and this leads us to a precise choice of the parameter $K>0$. In particular, we proceed as in Section \ref{ssec:apr} and let $K>0$ be such that $1+\|\phi+\psi\|_{L^\infty(\rho)}= 2^{(1-\kappa-\alpha)K}$. Then  in view of the embedding \eqref{eq:emb}}
\begin{equation}\label{eq:10}
\|\phi\|_{L^\infty(\rho)}+ \|\phi\|_{\CC^\alpha(\rho)}\lesssim 1,
\end{equation}
  \rmbb{where the right hand side is independent of $ K$.} So we have that
  \begin{equation}\label{eq:5.15}
  2^{(1-\kappa-\alpha)K}\lesssim 1+\|\psi\|_{L^\infty(\rho)}.
  \end{equation}
 \rmbb{The parameter $K$ as well as the corresponding localizers given according to Table \ref{t:loc} remain fixed for the rest of the analysis of the elliptic $\Phi^{4}_{5}$ model}. }

%

\subsection{Bound for $\phi$ in $\CC^{\frac{1}{2}+\alpha}(\rho^{\frac{3}{2}+\alpha})$}
\label{ssec:phi2}

As the next step, we estimate $\Phi$ in $\CC^{-\frac{3}{2}+\alpha}(\rho^{\frac{3}{2}+\alpha})$. It will be seen below that all the terms except for the one which is quadratic in $\phi+\psi$ can be even estimated in $\CC^{-\frac{3}{2}+\alpha}(\rho)$. 
 \rmbb{Recall that the parameter $K>0$  fixed in the previous section  determined the value of $L$ for each application of the localization operators, see Table \ref{t:loc}. }
In view of Lemma \ref{lem:local}, the sequence $(L_k)_{k\in\N_0}$ is therefore also fixed (and possibly different for each application of the  localizing operators). \rmbb{Below, we use the bound \eqref{eq:5.15} in order to control various powers of $2^{ K}$ by  $\|\psi\|^{\varepsilon}_{L^{\infty}(\rho)}$, where  $\varepsilon\in(0,1)$ is a generic constant \rmbb{independent of $K$} whose value changes from line to line. Accordingly, Lemma \ref{lem:local} yields}
$$
\|\UU_> \llbracket X^2 \rrbracket\|_{\CC^{-\frac{3}{2}+\alpha}}\lesssim 2^{-(\frac{1}{2}-\alpha-\kappa)K} \|\llbracket X^2 \rrbracket\|_{\CC^{-1-\kappa}(\rho^\sigma)},
$$
$$
\|\UU_> X^{\!\resizebox{!}{.8em}{
\begin{tikzpicture}
\pgfpathmoveto{\pgfqpoint{0cm}{-0.035cm}}
\pgfpathlineto{\pgfqpoint{1.976cm}{-0.035cm}}
\pgfpathlineto{\pgfqpoint{1.976cm}{1.94cm}}
\pgfpathlineto{\pgfqpoint{0cm}{1.94cm}}
\pgfpathclose
\pgfusepath{clip}
\begin{pgfscope}
\begin{pgfscope}
\pgfpathmoveto{\pgfqpoint{0cm}{-0.035cm}}
\pgfpathlineto{\pgfqpoint{1.976cm}{-0.035cm}}
\pgfpathlineto{\pgfqpoint{1.976cm}{1.94cm}}
\pgfpathlineto{\pgfqpoint{0cm}{1.94cm}}
\pgfpathclose
\pgfusepath{clip}
\begin{pgfscope}
\begin{pgfscope}
\pgfsetdash{}{0cm}
\pgfsetlinewidth{0.818mm}
\pgfsetroundcap
\pgfsetroundjoin
\pgfsetmiterlimit{7.0}
\definecolor{eps2pgf_color}{gray}{0}\pgfsetstrokecolor{eps2pgf_color}\pgfsetfillcolor{eps2pgf_color}
\pgfpathmoveto{\pgfqpoint{0.117cm}{1.815cm}}
\pgfpathlineto{\pgfqpoint{0.682cm}{1.065cm}}
\pgfpathlineto{\pgfqpoint{1.246cm}{1.815cm}}
\pgfusepath{stroke}
\end{pgfscope}
\definecolor{eps2pgf_color}{gray}{0}\pgfsetstrokecolor{eps2pgf_color}\pgfsetfillcolor{eps2pgf_color}
\pgfpathmoveto{\pgfqpoint{0.273cm}{1.789cm}}
\pgfpathcurveto{\pgfqpoint{0.273cm}{1.825cm}}{\pgfqpoint{0.259cm}{1.86cm}}{\pgfqpoint{0.233cm}{1.886cm}}
\pgfpathcurveto{\pgfqpoint{0.207cm}{1.912cm}}{\pgfqpoint{0.173cm}{1.926cm}}{\pgfqpoint{0.137cm}{1.926cm}}
\pgfpathcurveto{\pgfqpoint{0.1cm}{1.926cm}}{\pgfqpoint{0.066cm}{1.912cm}}{\pgfqpoint{0.04cm}{1.886cm}}
\pgfpathcurveto{\pgfqpoint{0.014cm}{1.86cm}}{\pgfqpoint{0cm}{1.825cm}}{\pgfqpoint{0cm}{1.789cm}}
\pgfpathcurveto{\pgfqpoint{0cm}{1.753cm}}{\pgfqpoint{0.014cm}{1.718cm}}{\pgfqpoint{0.04cm}{1.692cm}}
\pgfpathcurveto{\pgfqpoint{0.066cm}{1.667cm}}{\pgfqpoint{0.1cm}{1.652cm}}{\pgfqpoint{0.137cm}{1.652cm}}
\pgfpathcurveto{\pgfqpoint{0.173cm}{1.652cm}}{\pgfqpoint{0.207cm}{1.667cm}}{\pgfqpoint{0.233cm}{1.692cm}}
\pgfpathcurveto{\pgfqpoint{0.259cm}{1.718cm}}{\pgfqpoint{0.273cm}{1.753cm}}{\pgfqpoint{0.273cm}{1.789cm}}
\pgfusepath{fill}
\pgfpathmoveto{\pgfqpoint{1.345cm}{1.765cm}}
\pgfpathcurveto{\pgfqpoint{1.345cm}{1.801cm}}{\pgfqpoint{1.331cm}{1.836cm}}{\pgfqpoint{1.305cm}{1.862cm}}
\pgfpathcurveto{\pgfqpoint{1.28cm}{1.887cm}}{\pgfqpoint{1.245cm}{1.902cm}}{\pgfqpoint{1.209cm}{1.902cm}}
\pgfpathcurveto{\pgfqpoint{1.172cm}{1.902cm}}{\pgfqpoint{1.138cm}{1.887cm}}{\pgfqpoint{1.112cm}{1.862cm}}
\pgfpathcurveto{\pgfqpoint{1.087cm}{1.836cm}}{\pgfqpoint{1.072cm}{1.801cm}}{\pgfqpoint{1.072cm}{1.765cm}}
\pgfpathcurveto{\pgfqpoint{1.072cm}{1.728cm}}{\pgfqpoint{1.087cm}{1.694cm}}{\pgfqpoint{1.112cm}{1.668cm}}
\pgfpathcurveto{\pgfqpoint{1.138cm}{1.642cm}}{\pgfqpoint{1.172cm}{1.628cm}}{\pgfqpoint{1.209cm}{1.628cm}}
\pgfpathcurveto{\pgfqpoint{1.245cm}{1.628cm}}{\pgfqpoint{1.28cm}{1.642cm}}{\pgfqpoint{1.305cm}{1.668cm}}
\pgfpathcurveto{\pgfqpoint{1.331cm}{1.694cm}}{\pgfqpoint{1.345cm}{1.728cm}}{\pgfqpoint{1.345cm}{1.765cm}}
\pgfusepath{fill}
\begin{pgfscope}
\pgfsetdash{}{0cm}
\pgfsetlinewidth{0.818mm}
\pgfsetroundcap
\pgfsetroundjoin
\pgfsetmiterlimit{7.0}
\pgfpathmoveto{\pgfqpoint{0.682cm}{1.065cm}}
\pgfpathlineto{\pgfqpoint{1.246cm}{0.315cm}}
\pgfpathlineto{\pgfqpoint{1.811cm}{1.065cm}}
\pgfusepath{stroke}
\end{pgfscope}
\pgfpathmoveto{\pgfqpoint{1.948cm}{1.065cm}}
\pgfpathcurveto{\pgfqpoint{1.948cm}{1.101cm}}{\pgfqpoint{1.933cm}{1.136cm}}{\pgfqpoint{1.907cm}{1.162cm}}
\pgfpathcurveto{\pgfqpoint{1.882cm}{1.187cm}}{\pgfqpoint{1.847cm}{1.202cm}}{\pgfqpoint{1.811cm}{1.202cm}}
\pgfpathcurveto{\pgfqpoint{1.775cm}{1.202cm}}{\pgfqpoint{1.74cm}{1.187cm}}{\pgfqpoint{1.714cm}{1.162cm}}
\pgfpathcurveto{\pgfqpoint{1.689cm}{1.136cm}}{\pgfqpoint{1.674cm}{1.101cm}}{\pgfqpoint{1.674cm}{1.065cm}}
\pgfpathcurveto{\pgfqpoint{1.674cm}{1.029cm}}{\pgfqpoint{1.689cm}{0.994cm}}{\pgfqpoint{1.714cm}{0.968cm}}
\pgfpathcurveto{\pgfqpoint{1.74cm}{0.942cm}}{\pgfqpoint{1.775cm}{0.928cm}}{\pgfqpoint{1.811cm}{0.928cm}}
\pgfpathcurveto{\pgfqpoint{1.847cm}{0.928cm}}{\pgfqpoint{1.882cm}{0.942cm}}{\pgfqpoint{1.907cm}{0.968cm}}
\pgfpathcurveto{\pgfqpoint{1.933cm}{0.994cm}}{\pgfqpoint{1.948cm}{1.029cm}}{\pgfqpoint{1.948cm}{1.065cm}}
\pgfusepath{fill}
\begin{pgfscope}
\pgfsetdash{}{0cm}
\pgfsetlinewidth{0.818mm}
\pgfsetmiterlimit{7.0}
\pgfpathmoveto{\pgfqpoint{1.246cm}{0.315cm}}
\pgfpathlineto{\pgfqpoint{1.244cm}{1.061cm}}
\pgfusepath{stroke}
\end{pgfscope}
\pgfpathmoveto{\pgfqpoint{1.38cm}{1.065cm}}
\pgfpathcurveto{\pgfqpoint{1.38cm}{1.101cm}}{\pgfqpoint{1.366cm}{1.136cm}}{\pgfqpoint{1.34cm}{1.162cm}}
\pgfpathcurveto{\pgfqpoint{1.315cm}{1.187cm}}{\pgfqpoint{1.28cm}{1.202cm}}{\pgfqpoint{1.244cm}{1.202cm}}
\pgfpathcurveto{\pgfqpoint{1.207cm}{1.202cm}}{\pgfqpoint{1.173cm}{1.187cm}}{\pgfqpoint{1.147cm}{1.162cm}}
\pgfpathcurveto{\pgfqpoint{1.121cm}{1.136cm}}{\pgfqpoint{1.107cm}{1.101cm}}{\pgfqpoint{1.107cm}{1.065cm}}
\pgfpathcurveto{\pgfqpoint{1.107cm}{1.029cm}}{\pgfqpoint{1.121cm}{0.994cm}}{\pgfqpoint{1.147cm}{0.968cm}}
\pgfpathcurveto{\pgfqpoint{1.173cm}{0.942cm}}{\pgfqpoint{1.207cm}{0.928cm}}{\pgfqpoint{1.244cm}{0.928cm}}
\pgfpathcurveto{\pgfqpoint{1.28cm}{0.928cm}}{\pgfqpoint{1.315cm}{0.942cm}}{\pgfqpoint{1.34cm}{0.968cm}}
\pgfpathcurveto{\pgfqpoint{1.366cm}{0.994cm}}{\pgfqpoint{1.38cm}{1.029cm}}{\pgfqpoint{1.38cm}{1.065cm}}
\pgfusepath{fill}
\begin{pgfscope}
\pgfsetdash{}{0cm}
\pgfsetlinewidth{0.818mm}
\pgfsetmiterlimit{4.0}
\pgfpathmoveto{\pgfqpoint{1.383cm}{0.178cm}}
\pgfpathcurveto{\pgfqpoint{1.383cm}{0.214cm}}{\pgfqpoint{1.369cm}{0.249cm}}{\pgfqpoint{1.343cm}{0.275cm}}
\pgfpathcurveto{\pgfqpoint{1.317cm}{0.3cm}}{\pgfqpoint{1.283cm}{0.315cm}}{\pgfqpoint{1.246cm}{0.315cm}}
\pgfpathcurveto{\pgfqpoint{1.21cm}{0.315cm}}{\pgfqpoint{1.175cm}{0.3cm}}{\pgfqpoint{1.15cm}{0.275cm}}
\pgfpathcurveto{\pgfqpoint{1.124cm}{0.249cm}}{\pgfqpoint{1.11cm}{0.214cm}}{\pgfqpoint{1.11cm}{0.178cm}}
\pgfpathcurveto{\pgfqpoint{1.11cm}{0.141cm}}{\pgfqpoint{1.124cm}{0.107cm}}{\pgfqpoint{1.15cm}{0.081cm}}
\pgfpathcurveto{\pgfqpoint{1.175cm}{0.055cm}}{\pgfqpoint{1.21cm}{0.041cm}}{\pgfqpoint{1.246cm}{0.041cm}}
\pgfpathcurveto{\pgfqpoint{1.283cm}{0.041cm}}{\pgfqpoint{1.317cm}{0.055cm}}{\pgfqpoint{1.343cm}{0.081cm}}
\pgfpathcurveto{\pgfqpoint{1.369cm}{0.107cm}}{\pgfqpoint{1.383cm}{0.141cm}}{\pgfqpoint{1.383cm}{0.178cm}}
\pgfusepath{stroke}
\end{pgfscope}
\end{pgfscope}
\end{pgfscope}
\end{pgfscope}
\end{tikzpicture}}}\|_{\CC^{-\frac{3}{2}+\alpha}}\lesssim 2^{-(\frac{3}{2}-\alpha-\kappa)K/2} \|X^{\!\resizebox{!}{.8em}{
\begin{tikzpicture}
\pgfpathmoveto{\pgfqpoint{0cm}{-0.035cm}}
\pgfpathlineto{\pgfqpoint{1.976cm}{-0.035cm}}
\pgfpathlineto{\pgfqpoint{1.976cm}{1.94cm}}
\pgfpathlineto{\pgfqpoint{0cm}{1.94cm}}
\pgfpathclose
\pgfusepath{clip}
\begin{pgfscope}
\begin{pgfscope}
\pgfpathmoveto{\pgfqpoint{0cm}{-0.035cm}}
\pgfpathlineto{\pgfqpoint{1.976cm}{-0.035cm}}
\pgfpathlineto{\pgfqpoint{1.976cm}{1.94cm}}
\pgfpathlineto{\pgfqpoint{0cm}{1.94cm}}
\pgfpathclose
\pgfusepath{clip}
\begin{pgfscope}
\begin{pgfscope}
\pgfsetdash{}{0cm}
\pgfsetlinewidth{0.818mm}
\pgfsetroundcap
\pgfsetroundjoin
\pgfsetmiterlimit{7.0}
\definecolor{eps2pgf_color}{gray}{0}\pgfsetstrokecolor{eps2pgf_color}\pgfsetfillcolor{eps2pgf_color}
\pgfpathmoveto{\pgfqpoint{0.117cm}{1.815cm}}
\pgfpathlineto{\pgfqpoint{0.682cm}{1.065cm}}
\pgfpathlineto{\pgfqpoint{1.246cm}{1.815cm}}
\pgfusepath{stroke}
\end{pgfscope}
\definecolor{eps2pgf_color}{gray}{0}\pgfsetstrokecolor{eps2pgf_color}\pgfsetfillcolor{eps2pgf_color}
\pgfpathmoveto{\pgfqpoint{0.273cm}{1.789cm}}
\pgfpathcurveto{\pgfqpoint{0.273cm}{1.825cm}}{\pgfqpoint{0.259cm}{1.86cm}}{\pgfqpoint{0.233cm}{1.886cm}}
\pgfpathcurveto{\pgfqpoint{0.207cm}{1.912cm}}{\pgfqpoint{0.173cm}{1.926cm}}{\pgfqpoint{0.137cm}{1.926cm}}
\pgfpathcurveto{\pgfqpoint{0.1cm}{1.926cm}}{\pgfqpoint{0.066cm}{1.912cm}}{\pgfqpoint{0.04cm}{1.886cm}}
\pgfpathcurveto{\pgfqpoint{0.014cm}{1.86cm}}{\pgfqpoint{0cm}{1.825cm}}{\pgfqpoint{0cm}{1.789cm}}
\pgfpathcurveto{\pgfqpoint{0cm}{1.753cm}}{\pgfqpoint{0.014cm}{1.718cm}}{\pgfqpoint{0.04cm}{1.692cm}}
\pgfpathcurveto{\pgfqpoint{0.066cm}{1.667cm}}{\pgfqpoint{0.1cm}{1.652cm}}{\pgfqpoint{0.137cm}{1.652cm}}
\pgfpathcurveto{\pgfqpoint{0.173cm}{1.652cm}}{\pgfqpoint{0.207cm}{1.667cm}}{\pgfqpoint{0.233cm}{1.692cm}}
\pgfpathcurveto{\pgfqpoint{0.259cm}{1.718cm}}{\pgfqpoint{0.273cm}{1.753cm}}{\pgfqpoint{0.273cm}{1.789cm}}
\pgfusepath{fill}
\pgfpathmoveto{\pgfqpoint{1.345cm}{1.765cm}}
\pgfpathcurveto{\pgfqpoint{1.345cm}{1.801cm}}{\pgfqpoint{1.331cm}{1.836cm}}{\pgfqpoint{1.305cm}{1.862cm}}
\pgfpathcurveto{\pgfqpoint{1.28cm}{1.887cm}}{\pgfqpoint{1.245cm}{1.902cm}}{\pgfqpoint{1.209cm}{1.902cm}}
\pgfpathcurveto{\pgfqpoint{1.172cm}{1.902cm}}{\pgfqpoint{1.138cm}{1.887cm}}{\pgfqpoint{1.112cm}{1.862cm}}
\pgfpathcurveto{\pgfqpoint{1.087cm}{1.836cm}}{\pgfqpoint{1.072cm}{1.801cm}}{\pgfqpoint{1.072cm}{1.765cm}}
\pgfpathcurveto{\pgfqpoint{1.072cm}{1.728cm}}{\pgfqpoint{1.087cm}{1.694cm}}{\pgfqpoint{1.112cm}{1.668cm}}
\pgfpathcurveto{\pgfqpoint{1.138cm}{1.642cm}}{\pgfqpoint{1.172cm}{1.628cm}}{\pgfqpoint{1.209cm}{1.628cm}}
\pgfpathcurveto{\pgfqpoint{1.245cm}{1.628cm}}{\pgfqpoint{1.28cm}{1.642cm}}{\pgfqpoint{1.305cm}{1.668cm}}
\pgfpathcurveto{\pgfqpoint{1.331cm}{1.694cm}}{\pgfqpoint{1.345cm}{1.728cm}}{\pgfqpoint{1.345cm}{1.765cm}}
\pgfusepath{fill}
\begin{pgfscope}
\pgfsetdash{}{0cm}
\pgfsetlinewidth{0.818mm}
\pgfsetroundcap
\pgfsetroundjoin
\pgfsetmiterlimit{7.0}
\pgfpathmoveto{\pgfqpoint{0.682cm}{1.065cm}}
\pgfpathlineto{\pgfqpoint{1.246cm}{0.315cm}}
\pgfpathlineto{\pgfqpoint{1.811cm}{1.065cm}}
\pgfusepath{stroke}
\end{pgfscope}
\pgfpathmoveto{\pgfqpoint{1.948cm}{1.065cm}}
\pgfpathcurveto{\pgfqpoint{1.948cm}{1.101cm}}{\pgfqpoint{1.933cm}{1.136cm}}{\pgfqpoint{1.907cm}{1.162cm}}
\pgfpathcurveto{\pgfqpoint{1.882cm}{1.187cm}}{\pgfqpoint{1.847cm}{1.202cm}}{\pgfqpoint{1.811cm}{1.202cm}}
\pgfpathcurveto{\pgfqpoint{1.775cm}{1.202cm}}{\pgfqpoint{1.74cm}{1.187cm}}{\pgfqpoint{1.714cm}{1.162cm}}
\pgfpathcurveto{\pgfqpoint{1.689cm}{1.136cm}}{\pgfqpoint{1.674cm}{1.101cm}}{\pgfqpoint{1.674cm}{1.065cm}}
\pgfpathcurveto{\pgfqpoint{1.674cm}{1.029cm}}{\pgfqpoint{1.689cm}{0.994cm}}{\pgfqpoint{1.714cm}{0.968cm}}
\pgfpathcurveto{\pgfqpoint{1.74cm}{0.942cm}}{\pgfqpoint{1.775cm}{0.928cm}}{\pgfqpoint{1.811cm}{0.928cm}}
\pgfpathcurveto{\pgfqpoint{1.847cm}{0.928cm}}{\pgfqpoint{1.882cm}{0.942cm}}{\pgfqpoint{1.907cm}{0.968cm}}
\pgfpathcurveto{\pgfqpoint{1.933cm}{0.994cm}}{\pgfqpoint{1.948cm}{1.029cm}}{\pgfqpoint{1.948cm}{1.065cm}}
\pgfusepath{fill}
\begin{pgfscope}
\pgfsetdash{}{0cm}
\pgfsetlinewidth{0.818mm}
\pgfsetmiterlimit{7.0}
\pgfpathmoveto{\pgfqpoint{1.246cm}{0.315cm}}
\pgfpathlineto{\pgfqpoint{1.244cm}{1.061cm}}
\pgfusepath{stroke}
\end{pgfscope}
\pgfpathmoveto{\pgfqpoint{1.38cm}{1.065cm}}
\pgfpathcurveto{\pgfqpoint{1.38cm}{1.101cm}}{\pgfqpoint{1.366cm}{1.136cm}}{\pgfqpoint{1.34cm}{1.162cm}}
\pgfpathcurveto{\pgfqpoint{1.315cm}{1.187cm}}{\pgfqpoint{1.28cm}{1.202cm}}{\pgfqpoint{1.244cm}{1.202cm}}
\pgfpathcurveto{\pgfqpoint{1.207cm}{1.202cm}}{\pgfqpoint{1.173cm}{1.187cm}}{\pgfqpoint{1.147cm}{1.162cm}}
\pgfpathcurveto{\pgfqpoint{1.121cm}{1.136cm}}{\pgfqpoint{1.107cm}{1.101cm}}{\pgfqpoint{1.107cm}{1.065cm}}
\pgfpathcurveto{\pgfqpoint{1.107cm}{1.029cm}}{\pgfqpoint{1.121cm}{0.994cm}}{\pgfqpoint{1.147cm}{0.968cm}}
\pgfpathcurveto{\pgfqpoint{1.173cm}{0.942cm}}{\pgfqpoint{1.207cm}{0.928cm}}{\pgfqpoint{1.244cm}{0.928cm}}
\pgfpathcurveto{\pgfqpoint{1.28cm}{0.928cm}}{\pgfqpoint{1.315cm}{0.942cm}}{\pgfqpoint{1.34cm}{0.968cm}}
\pgfpathcurveto{\pgfqpoint{1.366cm}{0.994cm}}{\pgfqpoint{1.38cm}{1.029cm}}{\pgfqpoint{1.38cm}{1.065cm}}
\pgfusepath{fill}
\begin{pgfscope}
\pgfsetdash{}{0cm}
\pgfsetlinewidth{0.818mm}
\pgfsetmiterlimit{4.0}
\pgfpathmoveto{\pgfqpoint{1.383cm}{0.178cm}}
\pgfpathcurveto{\pgfqpoint{1.383cm}{0.214cm}}{\pgfqpoint{1.369cm}{0.249cm}}{\pgfqpoint{1.343cm}{0.275cm}}
\pgfpathcurveto{\pgfqpoint{1.317cm}{0.3cm}}{\pgfqpoint{1.283cm}{0.315cm}}{\pgfqpoint{1.246cm}{0.315cm}}
\pgfpathcurveto{\pgfqpoint{1.21cm}{0.315cm}}{\pgfqpoint{1.175cm}{0.3cm}}{\pgfqpoint{1.15cm}{0.275cm}}
\pgfpathcurveto{\pgfqpoint{1.124cm}{0.249cm}}{\pgfqpoint{1.11cm}{0.214cm}}{\pgfqpoint{1.11cm}{0.178cm}}
\pgfpathcurveto{\pgfqpoint{1.11cm}{0.141cm}}{\pgfqpoint{1.124cm}{0.107cm}}{\pgfqpoint{1.15cm}{0.081cm}}
\pgfpathcurveto{\pgfqpoint{1.175cm}{0.055cm}}{\pgfqpoint{1.21cm}{0.041cm}}{\pgfqpoint{1.246cm}{0.041cm}}
\pgfpathcurveto{\pgfqpoint{1.283cm}{0.041cm}}{\pgfqpoint{1.317cm}{0.055cm}}{\pgfqpoint{1.343cm}{0.081cm}}
\pgfpathcurveto{\pgfqpoint{1.369cm}{0.107cm}}{\pgfqpoint{1.383cm}{0.141cm}}{\pgfqpoint{1.383cm}{0.178cm}}
\pgfusepath{stroke}
\end{pgfscope}
\end{pgfscope}
\end{pgfscope}
\end{pgfscope}
\end{tikzpicture}}} \|_{\CC^{-\kappa}(\rho^\sigma)},
$$
$$
\|\UU_> X\|_{\CC^{-\frac{3}{2}+\alpha}(\rho^{-\alpha})}\lesssim 2^{-(1-\alpha-\kappa)\frac{2}{3}K} \|X\|_{\CC^{-\frac{1}{2}-\kappa}(\rho^\sigma)},
$$
$$
\|\UU_> X^{\!\resizebox{!}{.8em}{
\begin{tikzpicture}
\pgfpathmoveto{\pgfqpoint{0cm}{-0.035cm}}
\pgfpathlineto{\pgfqpoint{1.976cm}{-0.035cm}}
\pgfpathlineto{\pgfqpoint{1.976cm}{1.94cm}}
\pgfpathlineto{\pgfqpoint{0cm}{1.94cm}}
\pgfpathclose
\pgfusepath{clip}
\begin{pgfscope}
\begin{pgfscope}
\pgfpathmoveto{\pgfqpoint{0cm}{-0.035cm}}
\pgfpathlineto{\pgfqpoint{1.976cm}{-0.035cm}}
\pgfpathlineto{\pgfqpoint{1.976cm}{1.94cm}}
\pgfpathlineto{\pgfqpoint{0cm}{1.94cm}}
\pgfpathclose
\pgfusepath{clip}
\begin{pgfscope}
\begin{pgfscope}
\pgfsetdash{}{0cm}
\pgfsetlinewidth{0.818mm}
\pgfsetroundcap
\pgfsetroundjoin
\pgfsetmiterlimit{7.0}
\definecolor{eps2pgf_color}{gray}{0}\pgfsetstrokecolor{eps2pgf_color}\pgfsetfillcolor{eps2pgf_color}
\pgfpathmoveto{\pgfqpoint{0.117cm}{1.815cm}}
\pgfpathlineto{\pgfqpoint{0.682cm}{1.065cm}}
\pgfpathlineto{\pgfqpoint{1.246cm}{1.815cm}}
\pgfusepath{stroke}
\end{pgfscope}
\definecolor{eps2pgf_color}{gray}{0}\pgfsetstrokecolor{eps2pgf_color}\pgfsetfillcolor{eps2pgf_color}
\pgfpathmoveto{\pgfqpoint{0.273cm}{1.789cm}}
\pgfpathcurveto{\pgfqpoint{0.273cm}{1.825cm}}{\pgfqpoint{0.259cm}{1.86cm}}{\pgfqpoint{0.233cm}{1.886cm}}
\pgfpathcurveto{\pgfqpoint{0.207cm}{1.912cm}}{\pgfqpoint{0.173cm}{1.926cm}}{\pgfqpoint{0.137cm}{1.926cm}}
\pgfpathcurveto{\pgfqpoint{0.1cm}{1.926cm}}{\pgfqpoint{0.066cm}{1.912cm}}{\pgfqpoint{0.04cm}{1.886cm}}
\pgfpathcurveto{\pgfqpoint{0.014cm}{1.86cm}}{\pgfqpoint{0cm}{1.825cm}}{\pgfqpoint{0cm}{1.789cm}}
\pgfpathcurveto{\pgfqpoint{0cm}{1.753cm}}{\pgfqpoint{0.014cm}{1.718cm}}{\pgfqpoint{0.04cm}{1.692cm}}
\pgfpathcurveto{\pgfqpoint{0.066cm}{1.667cm}}{\pgfqpoint{0.1cm}{1.652cm}}{\pgfqpoint{0.137cm}{1.652cm}}
\pgfpathcurveto{\pgfqpoint{0.173cm}{1.652cm}}{\pgfqpoint{0.207cm}{1.667cm}}{\pgfqpoint{0.233cm}{1.692cm}}
\pgfpathcurveto{\pgfqpoint{0.259cm}{1.718cm}}{\pgfqpoint{0.273cm}{1.753cm}}{\pgfqpoint{0.273cm}{1.789cm}}
\pgfusepath{fill}
\begin{pgfscope}
\pgfsetdash{}{0cm}
\pgfsetlinewidth{0.818mm}
\pgfsetmiterlimit{7.0}
\pgfpathmoveto{\pgfqpoint{0.682cm}{1.065cm}}
\pgfpathlineto{\pgfqpoint{0.679cm}{1.812cm}}
\pgfusepath{stroke}
\end{pgfscope}
\pgfpathmoveto{\pgfqpoint{0.815cm}{1.793cm}}
\pgfpathcurveto{\pgfqpoint{0.815cm}{1.829cm}}{\pgfqpoint{0.801cm}{1.864cm}}{\pgfqpoint{0.775cm}{1.89cm}}
\pgfpathcurveto{\pgfqpoint{0.75cm}{1.915cm}}{\pgfqpoint{0.715cm}{1.93cm}}{\pgfqpoint{0.679cm}{1.93cm}}
\pgfpathcurveto{\pgfqpoint{0.643cm}{1.93cm}}{\pgfqpoint{0.608cm}{1.915cm}}{\pgfqpoint{0.582cm}{1.89cm}}
\pgfpathcurveto{\pgfqpoint{0.557cm}{1.864cm}}{\pgfqpoint{0.542cm}{1.829cm}}{\pgfqpoint{0.542cm}{1.793cm}}
\pgfpathcurveto{\pgfqpoint{0.542cm}{1.756cm}}{\pgfqpoint{0.557cm}{1.722cm}}{\pgfqpoint{0.582cm}{1.696cm}}
\pgfpathcurveto{\pgfqpoint{0.608cm}{1.67cm}}{\pgfqpoint{0.643cm}{1.656cm}}{\pgfqpoint{0.679cm}{1.656cm}}
\pgfpathcurveto{\pgfqpoint{0.715cm}{1.656cm}}{\pgfqpoint{0.75cm}{1.67cm}}{\pgfqpoint{0.775cm}{1.696cm}}
\pgfpathcurveto{\pgfqpoint{0.801cm}{1.722cm}}{\pgfqpoint{0.815cm}{1.756cm}}{\pgfqpoint{0.815cm}{1.793cm}}
\pgfusepath{fill}
\pgfpathmoveto{\pgfqpoint{1.345cm}{1.765cm}}
\pgfpathcurveto{\pgfqpoint{1.345cm}{1.801cm}}{\pgfqpoint{1.331cm}{1.836cm}}{\pgfqpoint{1.305cm}{1.862cm}}
\pgfpathcurveto{\pgfqpoint{1.28cm}{1.887cm}}{\pgfqpoint{1.245cm}{1.902cm}}{\pgfqpoint{1.209cm}{1.902cm}}
\pgfpathcurveto{\pgfqpoint{1.172cm}{1.902cm}}{\pgfqpoint{1.138cm}{1.887cm}}{\pgfqpoint{1.112cm}{1.862cm}}
\pgfpathcurveto{\pgfqpoint{1.087cm}{1.836cm}}{\pgfqpoint{1.072cm}{1.801cm}}{\pgfqpoint{1.072cm}{1.765cm}}
\pgfpathcurveto{\pgfqpoint{1.072cm}{1.728cm}}{\pgfqpoint{1.087cm}{1.694cm}}{\pgfqpoint{1.112cm}{1.668cm}}
\pgfpathcurveto{\pgfqpoint{1.138cm}{1.642cm}}{\pgfqpoint{1.172cm}{1.628cm}}{\pgfqpoint{1.209cm}{1.628cm}}
\pgfpathcurveto{\pgfqpoint{1.245cm}{1.628cm}}{\pgfqpoint{1.28cm}{1.642cm}}{\pgfqpoint{1.305cm}{1.668cm}}
\pgfpathcurveto{\pgfqpoint{1.331cm}{1.694cm}}{\pgfqpoint{1.345cm}{1.728cm}}{\pgfqpoint{1.345cm}{1.765cm}}
\pgfusepath{fill}
\begin{pgfscope}
\pgfsetdash{}{0cm}
\pgfsetlinewidth{0.818mm}
\pgfsetroundcap
\pgfsetroundjoin
\pgfsetmiterlimit{7.0}
\pgfpathmoveto{\pgfqpoint{0.682cm}{1.065cm}}
\pgfpathlineto{\pgfqpoint{1.246cm}{0.315cm}}
\pgfpathlineto{\pgfqpoint{1.811cm}{1.065cm}}
\pgfusepath{stroke}
\end{pgfscope}
\pgfpathmoveto{\pgfqpoint{1.948cm}{1.065cm}}
\pgfpathcurveto{\pgfqpoint{1.948cm}{1.101cm}}{\pgfqpoint{1.933cm}{1.136cm}}{\pgfqpoint{1.907cm}{1.162cm}}
\pgfpathcurveto{\pgfqpoint{1.882cm}{1.187cm}}{\pgfqpoint{1.847cm}{1.202cm}}{\pgfqpoint{1.811cm}{1.202cm}}
\pgfpathcurveto{\pgfqpoint{1.775cm}{1.202cm}}{\pgfqpoint{1.74cm}{1.187cm}}{\pgfqpoint{1.714cm}{1.162cm}}
\pgfpathcurveto{\pgfqpoint{1.689cm}{1.136cm}}{\pgfqpoint{1.674cm}{1.101cm}}{\pgfqpoint{1.674cm}{1.065cm}}
\pgfpathcurveto{\pgfqpoint{1.674cm}{1.029cm}}{\pgfqpoint{1.689cm}{0.994cm}}{\pgfqpoint{1.714cm}{0.968cm}}
\pgfpathcurveto{\pgfqpoint{1.74cm}{0.942cm}}{\pgfqpoint{1.775cm}{0.928cm}}{\pgfqpoint{1.811cm}{0.928cm}}
\pgfpathcurveto{\pgfqpoint{1.847cm}{0.928cm}}{\pgfqpoint{1.882cm}{0.942cm}}{\pgfqpoint{1.907cm}{0.968cm}}
\pgfpathcurveto{\pgfqpoint{1.933cm}{0.994cm}}{\pgfqpoint{1.948cm}{1.029cm}}{\pgfqpoint{1.948cm}{1.065cm}}
\pgfusepath{fill}
\begin{pgfscope}
\pgfsetdash{}{0cm}
\pgfsetlinewidth{0.818mm}
\pgfsetmiterlimit{4.0}
\pgfpathmoveto{\pgfqpoint{1.383cm}{0.178cm}}
\pgfpathcurveto{\pgfqpoint{1.383cm}{0.214cm}}{\pgfqpoint{1.369cm}{0.249cm}}{\pgfqpoint{1.343cm}{0.275cm}}
\pgfpathcurveto{\pgfqpoint{1.317cm}{0.3cm}}{\pgfqpoint{1.283cm}{0.315cm}}{\pgfqpoint{1.246cm}{0.315cm}}
\pgfpathcurveto{\pgfqpoint{1.21cm}{0.315cm}}{\pgfqpoint{1.175cm}{0.3cm}}{\pgfqpoint{1.15cm}{0.275cm}}
\pgfpathcurveto{\pgfqpoint{1.124cm}{0.249cm}}{\pgfqpoint{1.11cm}{0.214cm}}{\pgfqpoint{1.11cm}{0.178cm}}
\pgfpathcurveto{\pgfqpoint{1.11cm}{0.141cm}}{\pgfqpoint{1.124cm}{0.107cm}}{\pgfqpoint{1.15cm}{0.081cm}}
\pgfpathcurveto{\pgfqpoint{1.175cm}{0.055cm}}{\pgfqpoint{1.21cm}{0.041cm}}{\pgfqpoint{1.246cm}{0.041cm}}
\pgfpathcurveto{\pgfqpoint{1.283cm}{0.041cm}}{\pgfqpoint{1.317cm}{0.055cm}}{\pgfqpoint{1.343cm}{0.081cm}}
\pgfpathcurveto{\pgfqpoint{1.369cm}{0.107cm}}{\pgfqpoint{1.383cm}{0.141cm}}{\pgfqpoint{1.383cm}{0.178cm}}
\pgfusepath{stroke}
\end{pgfscope}
\end{pgfscope}
\end{pgfscope}
\end{pgfscope}
\end{tikzpicture}}}\|_{\CC^{-\frac{3}{2}+\alpha}}\lesssim 2^{-(\frac{3}{2}-\alpha-\kappa)K/2} \|X^{\!\resizebox{!}{.8em}{
\begin{tikzpicture}
\pgfpathmoveto{\pgfqpoint{0cm}{-0.035cm}}
\pgfpathlineto{\pgfqpoint{1.976cm}{-0.035cm}}
\pgfpathlineto{\pgfqpoint{1.976cm}{1.94cm}}
\pgfpathlineto{\pgfqpoint{0cm}{1.94cm}}
\pgfpathclose
\pgfusepath{clip}
\begin{pgfscope}
\begin{pgfscope}
\pgfpathmoveto{\pgfqpoint{0cm}{-0.035cm}}
\pgfpathlineto{\pgfqpoint{1.976cm}{-0.035cm}}
\pgfpathlineto{\pgfqpoint{1.976cm}{1.94cm}}
\pgfpathlineto{\pgfqpoint{0cm}{1.94cm}}
\pgfpathclose
\pgfusepath{clip}
\begin{pgfscope}
\begin{pgfscope}
\pgfsetdash{}{0cm}
\pgfsetlinewidth{0.818mm}
\pgfsetroundcap
\pgfsetroundjoin
\pgfsetmiterlimit{7.0}
\definecolor{eps2pgf_color}{gray}{0}\pgfsetstrokecolor{eps2pgf_color}\pgfsetfillcolor{eps2pgf_color}
\pgfpathmoveto{\pgfqpoint{0.117cm}{1.815cm}}
\pgfpathlineto{\pgfqpoint{0.682cm}{1.065cm}}
\pgfpathlineto{\pgfqpoint{1.246cm}{1.815cm}}
\pgfusepath{stroke}
\end{pgfscope}
\definecolor{eps2pgf_color}{gray}{0}\pgfsetstrokecolor{eps2pgf_color}\pgfsetfillcolor{eps2pgf_color}
\pgfpathmoveto{\pgfqpoint{0.273cm}{1.789cm}}
\pgfpathcurveto{\pgfqpoint{0.273cm}{1.825cm}}{\pgfqpoint{0.259cm}{1.86cm}}{\pgfqpoint{0.233cm}{1.886cm}}
\pgfpathcurveto{\pgfqpoint{0.207cm}{1.912cm}}{\pgfqpoint{0.173cm}{1.926cm}}{\pgfqpoint{0.137cm}{1.926cm}}
\pgfpathcurveto{\pgfqpoint{0.1cm}{1.926cm}}{\pgfqpoint{0.066cm}{1.912cm}}{\pgfqpoint{0.04cm}{1.886cm}}
\pgfpathcurveto{\pgfqpoint{0.014cm}{1.86cm}}{\pgfqpoint{0cm}{1.825cm}}{\pgfqpoint{0cm}{1.789cm}}
\pgfpathcurveto{\pgfqpoint{0cm}{1.753cm}}{\pgfqpoint{0.014cm}{1.718cm}}{\pgfqpoint{0.04cm}{1.692cm}}
\pgfpathcurveto{\pgfqpoint{0.066cm}{1.667cm}}{\pgfqpoint{0.1cm}{1.652cm}}{\pgfqpoint{0.137cm}{1.652cm}}
\pgfpathcurveto{\pgfqpoint{0.173cm}{1.652cm}}{\pgfqpoint{0.207cm}{1.667cm}}{\pgfqpoint{0.233cm}{1.692cm}}
\pgfpathcurveto{\pgfqpoint{0.259cm}{1.718cm}}{\pgfqpoint{0.273cm}{1.753cm}}{\pgfqpoint{0.273cm}{1.789cm}}
\pgfusepath{fill}
\begin{pgfscope}
\pgfsetdash{}{0cm}
\pgfsetlinewidth{0.818mm}
\pgfsetmiterlimit{7.0}
\pgfpathmoveto{\pgfqpoint{0.682cm}{1.065cm}}
\pgfpathlineto{\pgfqpoint{0.679cm}{1.812cm}}
\pgfusepath{stroke}
\end{pgfscope}
\pgfpathmoveto{\pgfqpoint{0.815cm}{1.793cm}}
\pgfpathcurveto{\pgfqpoint{0.815cm}{1.829cm}}{\pgfqpoint{0.801cm}{1.864cm}}{\pgfqpoint{0.775cm}{1.89cm}}
\pgfpathcurveto{\pgfqpoint{0.75cm}{1.915cm}}{\pgfqpoint{0.715cm}{1.93cm}}{\pgfqpoint{0.679cm}{1.93cm}}
\pgfpathcurveto{\pgfqpoint{0.643cm}{1.93cm}}{\pgfqpoint{0.608cm}{1.915cm}}{\pgfqpoint{0.582cm}{1.89cm}}
\pgfpathcurveto{\pgfqpoint{0.557cm}{1.864cm}}{\pgfqpoint{0.542cm}{1.829cm}}{\pgfqpoint{0.542cm}{1.793cm}}
\pgfpathcurveto{\pgfqpoint{0.542cm}{1.756cm}}{\pgfqpoint{0.557cm}{1.722cm}}{\pgfqpoint{0.582cm}{1.696cm}}
\pgfpathcurveto{\pgfqpoint{0.608cm}{1.67cm}}{\pgfqpoint{0.643cm}{1.656cm}}{\pgfqpoint{0.679cm}{1.656cm}}
\pgfpathcurveto{\pgfqpoint{0.715cm}{1.656cm}}{\pgfqpoint{0.75cm}{1.67cm}}{\pgfqpoint{0.775cm}{1.696cm}}
\pgfpathcurveto{\pgfqpoint{0.801cm}{1.722cm}}{\pgfqpoint{0.815cm}{1.756cm}}{\pgfqpoint{0.815cm}{1.793cm}}
\pgfusepath{fill}
\pgfpathmoveto{\pgfqpoint{1.345cm}{1.765cm}}
\pgfpathcurveto{\pgfqpoint{1.345cm}{1.801cm}}{\pgfqpoint{1.331cm}{1.836cm}}{\pgfqpoint{1.305cm}{1.862cm}}
\pgfpathcurveto{\pgfqpoint{1.28cm}{1.887cm}}{\pgfqpoint{1.245cm}{1.902cm}}{\pgfqpoint{1.209cm}{1.902cm}}
\pgfpathcurveto{\pgfqpoint{1.172cm}{1.902cm}}{\pgfqpoint{1.138cm}{1.887cm}}{\pgfqpoint{1.112cm}{1.862cm}}
\pgfpathcurveto{\pgfqpoint{1.087cm}{1.836cm}}{\pgfqpoint{1.072cm}{1.801cm}}{\pgfqpoint{1.072cm}{1.765cm}}
\pgfpathcurveto{\pgfqpoint{1.072cm}{1.728cm}}{\pgfqpoint{1.087cm}{1.694cm}}{\pgfqpoint{1.112cm}{1.668cm}}
\pgfpathcurveto{\pgfqpoint{1.138cm}{1.642cm}}{\pgfqpoint{1.172cm}{1.628cm}}{\pgfqpoint{1.209cm}{1.628cm}}
\pgfpathcurveto{\pgfqpoint{1.245cm}{1.628cm}}{\pgfqpoint{1.28cm}{1.642cm}}{\pgfqpoint{1.305cm}{1.668cm}}
\pgfpathcurveto{\pgfqpoint{1.331cm}{1.694cm}}{\pgfqpoint{1.345cm}{1.728cm}}{\pgfqpoint{1.345cm}{1.765cm}}
\pgfusepath{fill}
\begin{pgfscope}
\pgfsetdash{}{0cm}
\pgfsetlinewidth{0.818mm}
\pgfsetroundcap
\pgfsetroundjoin
\pgfsetmiterlimit{7.0}
\pgfpathmoveto{\pgfqpoint{0.682cm}{1.065cm}}
\pgfpathlineto{\pgfqpoint{1.246cm}{0.315cm}}
\pgfpathlineto{\pgfqpoint{1.811cm}{1.065cm}}
\pgfusepath{stroke}
\end{pgfscope}
\pgfpathmoveto{\pgfqpoint{1.948cm}{1.065cm}}
\pgfpathcurveto{\pgfqpoint{1.948cm}{1.101cm}}{\pgfqpoint{1.933cm}{1.136cm}}{\pgfqpoint{1.907cm}{1.162cm}}
\pgfpathcurveto{\pgfqpoint{1.882cm}{1.187cm}}{\pgfqpoint{1.847cm}{1.202cm}}{\pgfqpoint{1.811cm}{1.202cm}}
\pgfpathcurveto{\pgfqpoint{1.775cm}{1.202cm}}{\pgfqpoint{1.74cm}{1.187cm}}{\pgfqpoint{1.714cm}{1.162cm}}
\pgfpathcurveto{\pgfqpoint{1.689cm}{1.136cm}}{\pgfqpoint{1.674cm}{1.101cm}}{\pgfqpoint{1.674cm}{1.065cm}}
\pgfpathcurveto{\pgfqpoint{1.674cm}{1.029cm}}{\pgfqpoint{1.689cm}{0.994cm}}{\pgfqpoint{1.714cm}{0.968cm}}
\pgfpathcurveto{\pgfqpoint{1.74cm}{0.942cm}}{\pgfqpoint{1.775cm}{0.928cm}}{\pgfqpoint{1.811cm}{0.928cm}}
\pgfpathcurveto{\pgfqpoint{1.847cm}{0.928cm}}{\pgfqpoint{1.882cm}{0.942cm}}{\pgfqpoint{1.907cm}{0.968cm}}
\pgfpathcurveto{\pgfqpoint{1.933cm}{0.994cm}}{\pgfqpoint{1.948cm}{1.029cm}}{\pgfqpoint{1.948cm}{1.065cm}}
\pgfusepath{fill}
\begin{pgfscope}
\pgfsetdash{}{0cm}
\pgfsetlinewidth{0.818mm}
\pgfsetmiterlimit{4.0}
\pgfpathmoveto{\pgfqpoint{1.383cm}{0.178cm}}
\pgfpathcurveto{\pgfqpoint{1.383cm}{0.214cm}}{\pgfqpoint{1.369cm}{0.249cm}}{\pgfqpoint{1.343cm}{0.275cm}}
\pgfpathcurveto{\pgfqpoint{1.317cm}{0.3cm}}{\pgfqpoint{1.283cm}{0.315cm}}{\pgfqpoint{1.246cm}{0.315cm}}
\pgfpathcurveto{\pgfqpoint{1.21cm}{0.315cm}}{\pgfqpoint{1.175cm}{0.3cm}}{\pgfqpoint{1.15cm}{0.275cm}}
\pgfpathcurveto{\pgfqpoint{1.124cm}{0.249cm}}{\pgfqpoint{1.11cm}{0.214cm}}{\pgfqpoint{1.11cm}{0.178cm}}
\pgfpathcurveto{\pgfqpoint{1.11cm}{0.141cm}}{\pgfqpoint{1.124cm}{0.107cm}}{\pgfqpoint{1.15cm}{0.081cm}}
\pgfpathcurveto{\pgfqpoint{1.175cm}{0.055cm}}{\pgfqpoint{1.21cm}{0.041cm}}{\pgfqpoint{1.246cm}{0.041cm}}
\pgfpathcurveto{\pgfqpoint{1.283cm}{0.041cm}}{\pgfqpoint{1.317cm}{0.055cm}}{\pgfqpoint{1.343cm}{0.081cm}}
\pgfpathcurveto{\pgfqpoint{1.369cm}{0.107cm}}{\pgfqpoint{1.383cm}{0.141cm}}{\pgfqpoint{1.383cm}{0.178cm}}
\pgfusepath{stroke}
\end{pgfscope}
\end{pgfscope}
\end{pgfscope}
\end{pgfscope}
\end{tikzpicture}}} \|_{\CC^{-\kappa}(\rho^\sigma)},
$$
$$
\|\UU_> X\|_{\CC^{-\frac{3}{2}+\alpha}(\rho^{-\frac{1}{2}+\alpha})}\lesssim 2^{-(1-\alpha-\kappa)\frac{4}{3}K} \|X\|_{\CC^{-\frac{1}{2}-\kappa}(\rho^\sigma)},
$$
which implies
\begin{align*}
&\|3\UU_>\llbracket X^2 \rrbracket\succ(\phi+\psi)\|_{\CC^{-\frac{3}{2}+\alpha}(\rho)}+\|3\UU_>\llbracket X^2 \rrbracket\prec(\phi+\psi)\|_{\CC^{-\frac{3}{2}+\alpha}(\rho)}\\
&\quad\lesssim \|\phi+\psi\|_{L^\infty(\rho)}\|\UU_>\llbracket X^2 \rrbracket\|_{\CC^{-\frac{3}{2}+\alpha}}\lesssim 2^{-(\frac{1}{2}-\alpha-\kappa)K} \|\phi+\psi\|_{L^\infty(\rho)}\\
&\quad\lesssim 2^{-(\frac{1}{2}-\alpha-\kappa)K}2^{(1-\alpha-\kappa)K}\lesssim 1+\|\psi\|_{L^\infty(\rho)}^{\varepsilon},
\end{align*}
\begin{align*}
\|3( \phi + \psi)\prec\UU_>X^{\!\resizebox{!}{.8em}{
\begin{tikzpicture}
\pgfpathmoveto{\pgfqpoint{0cm}{-0.035cm}}
\pgfpathlineto{\pgfqpoint{1.976cm}{-0.035cm}}
\pgfpathlineto{\pgfqpoint{1.976cm}{1.94cm}}
\pgfpathlineto{\pgfqpoint{0cm}{1.94cm}}
\pgfpathclose
\pgfusepath{clip}
\begin{pgfscope}
\begin{pgfscope}
\pgfpathmoveto{\pgfqpoint{0cm}{-0.035cm}}
\pgfpathlineto{\pgfqpoint{1.976cm}{-0.035cm}}
\pgfpathlineto{\pgfqpoint{1.976cm}{1.94cm}}
\pgfpathlineto{\pgfqpoint{0cm}{1.94cm}}
\pgfpathclose
\pgfusepath{clip}
\begin{pgfscope}
\begin{pgfscope}
\pgfsetdash{}{0cm}
\pgfsetlinewidth{0.818mm}
\pgfsetroundcap
\pgfsetroundjoin
\pgfsetmiterlimit{7.0}
\definecolor{eps2pgf_color}{gray}{0}\pgfsetstrokecolor{eps2pgf_color}\pgfsetfillcolor{eps2pgf_color}
\pgfpathmoveto{\pgfqpoint{0.117cm}{1.815cm}}
\pgfpathlineto{\pgfqpoint{0.682cm}{1.065cm}}
\pgfpathlineto{\pgfqpoint{1.246cm}{1.815cm}}
\pgfusepath{stroke}
\end{pgfscope}
\definecolor{eps2pgf_color}{gray}{0}\pgfsetstrokecolor{eps2pgf_color}\pgfsetfillcolor{eps2pgf_color}
\pgfpathmoveto{\pgfqpoint{0.273cm}{1.789cm}}
\pgfpathcurveto{\pgfqpoint{0.273cm}{1.825cm}}{\pgfqpoint{0.259cm}{1.86cm}}{\pgfqpoint{0.233cm}{1.886cm}}
\pgfpathcurveto{\pgfqpoint{0.207cm}{1.912cm}}{\pgfqpoint{0.173cm}{1.926cm}}{\pgfqpoint{0.137cm}{1.926cm}}
\pgfpathcurveto{\pgfqpoint{0.1cm}{1.926cm}}{\pgfqpoint{0.066cm}{1.912cm}}{\pgfqpoint{0.04cm}{1.886cm}}
\pgfpathcurveto{\pgfqpoint{0.014cm}{1.86cm}}{\pgfqpoint{0cm}{1.825cm}}{\pgfqpoint{0cm}{1.789cm}}
\pgfpathcurveto{\pgfqpoint{0cm}{1.753cm}}{\pgfqpoint{0.014cm}{1.718cm}}{\pgfqpoint{0.04cm}{1.692cm}}
\pgfpathcurveto{\pgfqpoint{0.066cm}{1.667cm}}{\pgfqpoint{0.1cm}{1.652cm}}{\pgfqpoint{0.137cm}{1.652cm}}
\pgfpathcurveto{\pgfqpoint{0.173cm}{1.652cm}}{\pgfqpoint{0.207cm}{1.667cm}}{\pgfqpoint{0.233cm}{1.692cm}}
\pgfpathcurveto{\pgfqpoint{0.259cm}{1.718cm}}{\pgfqpoint{0.273cm}{1.753cm}}{\pgfqpoint{0.273cm}{1.789cm}}
\pgfusepath{fill}
\pgfpathmoveto{\pgfqpoint{1.345cm}{1.765cm}}
\pgfpathcurveto{\pgfqpoint{1.345cm}{1.801cm}}{\pgfqpoint{1.331cm}{1.836cm}}{\pgfqpoint{1.305cm}{1.862cm}}
\pgfpathcurveto{\pgfqpoint{1.28cm}{1.887cm}}{\pgfqpoint{1.245cm}{1.902cm}}{\pgfqpoint{1.209cm}{1.902cm}}
\pgfpathcurveto{\pgfqpoint{1.172cm}{1.902cm}}{\pgfqpoint{1.138cm}{1.887cm}}{\pgfqpoint{1.112cm}{1.862cm}}
\pgfpathcurveto{\pgfqpoint{1.087cm}{1.836cm}}{\pgfqpoint{1.072cm}{1.801cm}}{\pgfqpoint{1.072cm}{1.765cm}}
\pgfpathcurveto{\pgfqpoint{1.072cm}{1.728cm}}{\pgfqpoint{1.087cm}{1.694cm}}{\pgfqpoint{1.112cm}{1.668cm}}
\pgfpathcurveto{\pgfqpoint{1.138cm}{1.642cm}}{\pgfqpoint{1.172cm}{1.628cm}}{\pgfqpoint{1.209cm}{1.628cm}}
\pgfpathcurveto{\pgfqpoint{1.245cm}{1.628cm}}{\pgfqpoint{1.28cm}{1.642cm}}{\pgfqpoint{1.305cm}{1.668cm}}
\pgfpathcurveto{\pgfqpoint{1.331cm}{1.694cm}}{\pgfqpoint{1.345cm}{1.728cm}}{\pgfqpoint{1.345cm}{1.765cm}}
\pgfusepath{fill}
\begin{pgfscope}
\pgfsetdash{}{0cm}
\pgfsetlinewidth{0.818mm}
\pgfsetroundcap
\pgfsetroundjoin
\pgfsetmiterlimit{7.0}
\pgfpathmoveto{\pgfqpoint{0.682cm}{1.065cm}}
\pgfpathlineto{\pgfqpoint{1.246cm}{0.315cm}}
\pgfpathlineto{\pgfqpoint{1.811cm}{1.065cm}}
\pgfusepath{stroke}
\end{pgfscope}
\pgfpathmoveto{\pgfqpoint{1.948cm}{1.065cm}}
\pgfpathcurveto{\pgfqpoint{1.948cm}{1.101cm}}{\pgfqpoint{1.933cm}{1.136cm}}{\pgfqpoint{1.907cm}{1.162cm}}
\pgfpathcurveto{\pgfqpoint{1.882cm}{1.187cm}}{\pgfqpoint{1.847cm}{1.202cm}}{\pgfqpoint{1.811cm}{1.202cm}}
\pgfpathcurveto{\pgfqpoint{1.775cm}{1.202cm}}{\pgfqpoint{1.74cm}{1.187cm}}{\pgfqpoint{1.714cm}{1.162cm}}
\pgfpathcurveto{\pgfqpoint{1.689cm}{1.136cm}}{\pgfqpoint{1.674cm}{1.101cm}}{\pgfqpoint{1.674cm}{1.065cm}}
\pgfpathcurveto{\pgfqpoint{1.674cm}{1.029cm}}{\pgfqpoint{1.689cm}{0.994cm}}{\pgfqpoint{1.714cm}{0.968cm}}
\pgfpathcurveto{\pgfqpoint{1.74cm}{0.942cm}}{\pgfqpoint{1.775cm}{0.928cm}}{\pgfqpoint{1.811cm}{0.928cm}}
\pgfpathcurveto{\pgfqpoint{1.847cm}{0.928cm}}{\pgfqpoint{1.882cm}{0.942cm}}{\pgfqpoint{1.907cm}{0.968cm}}
\pgfpathcurveto{\pgfqpoint{1.933cm}{0.994cm}}{\pgfqpoint{1.948cm}{1.029cm}}{\pgfqpoint{1.948cm}{1.065cm}}
\pgfusepath{fill}
\begin{pgfscope}
\pgfsetdash{}{0cm}
\pgfsetlinewidth{0.818mm}
\pgfsetmiterlimit{7.0}
\pgfpathmoveto{\pgfqpoint{1.246cm}{0.315cm}}
\pgfpathlineto{\pgfqpoint{1.244cm}{1.061cm}}
\pgfusepath{stroke}
\end{pgfscope}
\pgfpathmoveto{\pgfqpoint{1.38cm}{1.065cm}}
\pgfpathcurveto{\pgfqpoint{1.38cm}{1.101cm}}{\pgfqpoint{1.366cm}{1.136cm}}{\pgfqpoint{1.34cm}{1.162cm}}
\pgfpathcurveto{\pgfqpoint{1.315cm}{1.187cm}}{\pgfqpoint{1.28cm}{1.202cm}}{\pgfqpoint{1.244cm}{1.202cm}}
\pgfpathcurveto{\pgfqpoint{1.207cm}{1.202cm}}{\pgfqpoint{1.173cm}{1.187cm}}{\pgfqpoint{1.147cm}{1.162cm}}
\pgfpathcurveto{\pgfqpoint{1.121cm}{1.136cm}}{\pgfqpoint{1.107cm}{1.101cm}}{\pgfqpoint{1.107cm}{1.065cm}}
\pgfpathcurveto{\pgfqpoint{1.107cm}{1.029cm}}{\pgfqpoint{1.121cm}{0.994cm}}{\pgfqpoint{1.147cm}{0.968cm}}
\pgfpathcurveto{\pgfqpoint{1.173cm}{0.942cm}}{\pgfqpoint{1.207cm}{0.928cm}}{\pgfqpoint{1.244cm}{0.928cm}}
\pgfpathcurveto{\pgfqpoint{1.28cm}{0.928cm}}{\pgfqpoint{1.315cm}{0.942cm}}{\pgfqpoint{1.34cm}{0.968cm}}
\pgfpathcurveto{\pgfqpoint{1.366cm}{0.994cm}}{\pgfqpoint{1.38cm}{1.029cm}}{\pgfqpoint{1.38cm}{1.065cm}}
\pgfusepath{fill}
\begin{pgfscope}
\pgfsetdash{}{0cm}
\pgfsetlinewidth{0.818mm}
\pgfsetmiterlimit{4.0}
\pgfpathmoveto{\pgfqpoint{1.383cm}{0.178cm}}
\pgfpathcurveto{\pgfqpoint{1.383cm}{0.214cm}}{\pgfqpoint{1.369cm}{0.249cm}}{\pgfqpoint{1.343cm}{0.275cm}}
\pgfpathcurveto{\pgfqpoint{1.317cm}{0.3cm}}{\pgfqpoint{1.283cm}{0.315cm}}{\pgfqpoint{1.246cm}{0.315cm}}
\pgfpathcurveto{\pgfqpoint{1.21cm}{0.315cm}}{\pgfqpoint{1.175cm}{0.3cm}}{\pgfqpoint{1.15cm}{0.275cm}}
\pgfpathcurveto{\pgfqpoint{1.124cm}{0.249cm}}{\pgfqpoint{1.11cm}{0.214cm}}{\pgfqpoint{1.11cm}{0.178cm}}
\pgfpathcurveto{\pgfqpoint{1.11cm}{0.141cm}}{\pgfqpoint{1.124cm}{0.107cm}}{\pgfqpoint{1.15cm}{0.081cm}}
\pgfpathcurveto{\pgfqpoint{1.175cm}{0.055cm}}{\pgfqpoint{1.21cm}{0.041cm}}{\pgfqpoint{1.246cm}{0.041cm}}
\pgfpathcurveto{\pgfqpoint{1.283cm}{0.041cm}}{\pgfqpoint{1.317cm}{0.055cm}}{\pgfqpoint{1.343cm}{0.081cm}}
\pgfpathcurveto{\pgfqpoint{1.369cm}{0.107cm}}{\pgfqpoint{1.383cm}{0.141cm}}{\pgfqpoint{1.383cm}{0.178cm}}
\pgfusepath{stroke}
\end{pgfscope}
\end{pgfscope}
\end{pgfscope}
\end{pgfscope}
\end{tikzpicture}}}\|_{\CC^{-\frac{3}{2}+\alpha}(\rho)}
&\lesssim  \|\phi+\psi\|_{L^\infty(\rho)}\|\UU_>X^{\!\resizebox{!}{.8em}{
\begin{tikzpicture}
\pgfpathmoveto{\pgfqpoint{0cm}{-0.035cm}}
\pgfpathlineto{\pgfqpoint{1.976cm}{-0.035cm}}
\pgfpathlineto{\pgfqpoint{1.976cm}{1.94cm}}
\pgfpathlineto{\pgfqpoint{0cm}{1.94cm}}
\pgfpathclose
\pgfusepath{clip}
\begin{pgfscope}
\begin{pgfscope}
\pgfpathmoveto{\pgfqpoint{0cm}{-0.035cm}}
\pgfpathlineto{\pgfqpoint{1.976cm}{-0.035cm}}
\pgfpathlineto{\pgfqpoint{1.976cm}{1.94cm}}
\pgfpathlineto{\pgfqpoint{0cm}{1.94cm}}
\pgfpathclose
\pgfusepath{clip}
\begin{pgfscope}
\begin{pgfscope}
\pgfsetdash{}{0cm}
\pgfsetlinewidth{0.818mm}
\pgfsetroundcap
\pgfsetroundjoin
\pgfsetmiterlimit{7.0}
\definecolor{eps2pgf_color}{gray}{0}\pgfsetstrokecolor{eps2pgf_color}\pgfsetfillcolor{eps2pgf_color}
\pgfpathmoveto{\pgfqpoint{0.117cm}{1.815cm}}
\pgfpathlineto{\pgfqpoint{0.682cm}{1.065cm}}
\pgfpathlineto{\pgfqpoint{1.246cm}{1.815cm}}
\pgfusepath{stroke}
\end{pgfscope}
\definecolor{eps2pgf_color}{gray}{0}\pgfsetstrokecolor{eps2pgf_color}\pgfsetfillcolor{eps2pgf_color}
\pgfpathmoveto{\pgfqpoint{0.273cm}{1.789cm}}
\pgfpathcurveto{\pgfqpoint{0.273cm}{1.825cm}}{\pgfqpoint{0.259cm}{1.86cm}}{\pgfqpoint{0.233cm}{1.886cm}}
\pgfpathcurveto{\pgfqpoint{0.207cm}{1.912cm}}{\pgfqpoint{0.173cm}{1.926cm}}{\pgfqpoint{0.137cm}{1.926cm}}
\pgfpathcurveto{\pgfqpoint{0.1cm}{1.926cm}}{\pgfqpoint{0.066cm}{1.912cm}}{\pgfqpoint{0.04cm}{1.886cm}}
\pgfpathcurveto{\pgfqpoint{0.014cm}{1.86cm}}{\pgfqpoint{0cm}{1.825cm}}{\pgfqpoint{0cm}{1.789cm}}
\pgfpathcurveto{\pgfqpoint{0cm}{1.753cm}}{\pgfqpoint{0.014cm}{1.718cm}}{\pgfqpoint{0.04cm}{1.692cm}}
\pgfpathcurveto{\pgfqpoint{0.066cm}{1.667cm}}{\pgfqpoint{0.1cm}{1.652cm}}{\pgfqpoint{0.137cm}{1.652cm}}
\pgfpathcurveto{\pgfqpoint{0.173cm}{1.652cm}}{\pgfqpoint{0.207cm}{1.667cm}}{\pgfqpoint{0.233cm}{1.692cm}}
\pgfpathcurveto{\pgfqpoint{0.259cm}{1.718cm}}{\pgfqpoint{0.273cm}{1.753cm}}{\pgfqpoint{0.273cm}{1.789cm}}
\pgfusepath{fill}
\pgfpathmoveto{\pgfqpoint{1.345cm}{1.765cm}}
\pgfpathcurveto{\pgfqpoint{1.345cm}{1.801cm}}{\pgfqpoint{1.331cm}{1.836cm}}{\pgfqpoint{1.305cm}{1.862cm}}
\pgfpathcurveto{\pgfqpoint{1.28cm}{1.887cm}}{\pgfqpoint{1.245cm}{1.902cm}}{\pgfqpoint{1.209cm}{1.902cm}}
\pgfpathcurveto{\pgfqpoint{1.172cm}{1.902cm}}{\pgfqpoint{1.138cm}{1.887cm}}{\pgfqpoint{1.112cm}{1.862cm}}
\pgfpathcurveto{\pgfqpoint{1.087cm}{1.836cm}}{\pgfqpoint{1.072cm}{1.801cm}}{\pgfqpoint{1.072cm}{1.765cm}}
\pgfpathcurveto{\pgfqpoint{1.072cm}{1.728cm}}{\pgfqpoint{1.087cm}{1.694cm}}{\pgfqpoint{1.112cm}{1.668cm}}
\pgfpathcurveto{\pgfqpoint{1.138cm}{1.642cm}}{\pgfqpoint{1.172cm}{1.628cm}}{\pgfqpoint{1.209cm}{1.628cm}}
\pgfpathcurveto{\pgfqpoint{1.245cm}{1.628cm}}{\pgfqpoint{1.28cm}{1.642cm}}{\pgfqpoint{1.305cm}{1.668cm}}
\pgfpathcurveto{\pgfqpoint{1.331cm}{1.694cm}}{\pgfqpoint{1.345cm}{1.728cm}}{\pgfqpoint{1.345cm}{1.765cm}}
\pgfusepath{fill}
\begin{pgfscope}
\pgfsetdash{}{0cm}
\pgfsetlinewidth{0.818mm}
\pgfsetroundcap
\pgfsetroundjoin
\pgfsetmiterlimit{7.0}
\pgfpathmoveto{\pgfqpoint{0.682cm}{1.065cm}}
\pgfpathlineto{\pgfqpoint{1.246cm}{0.315cm}}
\pgfpathlineto{\pgfqpoint{1.811cm}{1.065cm}}
\pgfusepath{stroke}
\end{pgfscope}
\pgfpathmoveto{\pgfqpoint{1.948cm}{1.065cm}}
\pgfpathcurveto{\pgfqpoint{1.948cm}{1.101cm}}{\pgfqpoint{1.933cm}{1.136cm}}{\pgfqpoint{1.907cm}{1.162cm}}
\pgfpathcurveto{\pgfqpoint{1.882cm}{1.187cm}}{\pgfqpoint{1.847cm}{1.202cm}}{\pgfqpoint{1.811cm}{1.202cm}}
\pgfpathcurveto{\pgfqpoint{1.775cm}{1.202cm}}{\pgfqpoint{1.74cm}{1.187cm}}{\pgfqpoint{1.714cm}{1.162cm}}
\pgfpathcurveto{\pgfqpoint{1.689cm}{1.136cm}}{\pgfqpoint{1.674cm}{1.101cm}}{\pgfqpoint{1.674cm}{1.065cm}}
\pgfpathcurveto{\pgfqpoint{1.674cm}{1.029cm}}{\pgfqpoint{1.689cm}{0.994cm}}{\pgfqpoint{1.714cm}{0.968cm}}
\pgfpathcurveto{\pgfqpoint{1.74cm}{0.942cm}}{\pgfqpoint{1.775cm}{0.928cm}}{\pgfqpoint{1.811cm}{0.928cm}}
\pgfpathcurveto{\pgfqpoint{1.847cm}{0.928cm}}{\pgfqpoint{1.882cm}{0.942cm}}{\pgfqpoint{1.907cm}{0.968cm}}
\pgfpathcurveto{\pgfqpoint{1.933cm}{0.994cm}}{\pgfqpoint{1.948cm}{1.029cm}}{\pgfqpoint{1.948cm}{1.065cm}}
\pgfusepath{fill}
\begin{pgfscope}
\pgfsetdash{}{0cm}
\pgfsetlinewidth{0.818mm}
\pgfsetmiterlimit{7.0}
\pgfpathmoveto{\pgfqpoint{1.246cm}{0.315cm}}
\pgfpathlineto{\pgfqpoint{1.244cm}{1.061cm}}
\pgfusepath{stroke}
\end{pgfscope}
\pgfpathmoveto{\pgfqpoint{1.38cm}{1.065cm}}
\pgfpathcurveto{\pgfqpoint{1.38cm}{1.101cm}}{\pgfqpoint{1.366cm}{1.136cm}}{\pgfqpoint{1.34cm}{1.162cm}}
\pgfpathcurveto{\pgfqpoint{1.315cm}{1.187cm}}{\pgfqpoint{1.28cm}{1.202cm}}{\pgfqpoint{1.244cm}{1.202cm}}
\pgfpathcurveto{\pgfqpoint{1.207cm}{1.202cm}}{\pgfqpoint{1.173cm}{1.187cm}}{\pgfqpoint{1.147cm}{1.162cm}}
\pgfpathcurveto{\pgfqpoint{1.121cm}{1.136cm}}{\pgfqpoint{1.107cm}{1.101cm}}{\pgfqpoint{1.107cm}{1.065cm}}
\pgfpathcurveto{\pgfqpoint{1.107cm}{1.029cm}}{\pgfqpoint{1.121cm}{0.994cm}}{\pgfqpoint{1.147cm}{0.968cm}}
\pgfpathcurveto{\pgfqpoint{1.173cm}{0.942cm}}{\pgfqpoint{1.207cm}{0.928cm}}{\pgfqpoint{1.244cm}{0.928cm}}
\pgfpathcurveto{\pgfqpoint{1.28cm}{0.928cm}}{\pgfqpoint{1.315cm}{0.942cm}}{\pgfqpoint{1.34cm}{0.968cm}}
\pgfpathcurveto{\pgfqpoint{1.366cm}{0.994cm}}{\pgfqpoint{1.38cm}{1.029cm}}{\pgfqpoint{1.38cm}{1.065cm}}
\pgfusepath{fill}
\begin{pgfscope}
\pgfsetdash{}{0cm}
\pgfsetlinewidth{0.818mm}
\pgfsetmiterlimit{4.0}
\pgfpathmoveto{\pgfqpoint{1.383cm}{0.178cm}}
\pgfpathcurveto{\pgfqpoint{1.383cm}{0.214cm}}{\pgfqpoint{1.369cm}{0.249cm}}{\pgfqpoint{1.343cm}{0.275cm}}
\pgfpathcurveto{\pgfqpoint{1.317cm}{0.3cm}}{\pgfqpoint{1.283cm}{0.315cm}}{\pgfqpoint{1.246cm}{0.315cm}}
\pgfpathcurveto{\pgfqpoint{1.21cm}{0.315cm}}{\pgfqpoint{1.175cm}{0.3cm}}{\pgfqpoint{1.15cm}{0.275cm}}
\pgfpathcurveto{\pgfqpoint{1.124cm}{0.249cm}}{\pgfqpoint{1.11cm}{0.214cm}}{\pgfqpoint{1.11cm}{0.178cm}}
\pgfpathcurveto{\pgfqpoint{1.11cm}{0.141cm}}{\pgfqpoint{1.124cm}{0.107cm}}{\pgfqpoint{1.15cm}{0.081cm}}
\pgfpathcurveto{\pgfqpoint{1.175cm}{0.055cm}}{\pgfqpoint{1.21cm}{0.041cm}}{\pgfqpoint{1.246cm}{0.041cm}}
\pgfpathcurveto{\pgfqpoint{1.283cm}{0.041cm}}{\pgfqpoint{1.317cm}{0.055cm}}{\pgfqpoint{1.343cm}{0.081cm}}
\pgfpathcurveto{\pgfqpoint{1.369cm}{0.107cm}}{\pgfqpoint{1.383cm}{0.141cm}}{\pgfqpoint{1.383cm}{0.178cm}}
\pgfusepath{stroke}
\end{pgfscope}
\end{pgfscope}
\end{pgfscope}
\end{pgfscope}
\end{tikzpicture}}}\|_{\CC^{-\frac{3}{2}+\alpha}}\lesssim 2^{-(\frac{3}{2}-\alpha-\kappa)K/2}\|\phi+\psi\|_{L^\infty(\rho)}\\
&\lesssim  1+\|\psi\|_{L^\infty(\rho)}^{\varepsilon},
\end{align*}
\begin{align*}
&\|6\UU_>X\succ(X^{\!\resizebox{0.6em}{!}{
\begin{tikzpicture}
\pgfpathmoveto{\pgfqpoint{0cm}{-0.035cm}}
\pgfpathlineto{\pgfqpoint{1.376cm}{-0.035cm}}
\pgfpathlineto{\pgfqpoint{1.376cm}{1.552cm}}
\pgfpathlineto{\pgfqpoint{0cm}{1.552cm}}
\pgfpathclose
\pgfusepath{clip}
\begin{pgfscope}
\begin{pgfscope}
\pgfpathmoveto{\pgfqpoint{0cm}{-0.035cm}}
\pgfpathlineto{\pgfqpoint{1.376cm}{-0.035cm}}
\pgfpathlineto{\pgfqpoint{1.376cm}{1.552cm}}
\pgfpathlineto{\pgfqpoint{0cm}{1.552cm}}
\pgfpathclose
\pgfusepath{clip}
\begin{pgfscope}
\begin{pgfscope}
\pgfsetdash{}{0cm}
\pgfsetlinewidth{0.818mm}
\pgfsetroundcap
\pgfsetroundjoin
\pgfsetmiterlimit{7.0}
\definecolor{eps2pgf_color}{gray}{0}\pgfsetstrokecolor{eps2pgf_color}\pgfsetfillcolor{eps2pgf_color}
\pgfpathmoveto{\pgfqpoint{0.117cm}{1.421cm}}
\pgfpathlineto{\pgfqpoint{0.682cm}{0.671cm}}
\pgfpathlineto{\pgfqpoint{1.246cm}{1.421cm}}
\pgfusepath{stroke}
\end{pgfscope}
\definecolor{eps2pgf_color}{gray}{0}\pgfsetstrokecolor{eps2pgf_color}\pgfsetfillcolor{eps2pgf_color}
\pgfpathmoveto{\pgfqpoint{0.273cm}{1.395cm}}
\pgfpathcurveto{\pgfqpoint{0.273cm}{1.432cm}}{\pgfqpoint{0.259cm}{1.467cm}}{\pgfqpoint{0.233cm}{1.492cm}}
\pgfpathcurveto{\pgfqpoint{0.207cm}{1.518cm}}{\pgfqpoint{0.173cm}{1.532cm}}{\pgfqpoint{0.137cm}{1.532cm}}
\pgfpathcurveto{\pgfqpoint{0.1cm}{1.532cm}}{\pgfqpoint{0.066cm}{1.518cm}}{\pgfqpoint{0.04cm}{1.492cm}}
\pgfpathcurveto{\pgfqpoint{0.014cm}{1.467cm}}{\pgfqpoint{0cm}{1.432cm}}{\pgfqpoint{0cm}{1.395cm}}
\pgfpathcurveto{\pgfqpoint{0cm}{1.359cm}}{\pgfqpoint{0.014cm}{1.324cm}}{\pgfqpoint{0.04cm}{1.299cm}}
\pgfpathcurveto{\pgfqpoint{0.066cm}{1.273cm}}{\pgfqpoint{0.1cm}{1.258cm}}{\pgfqpoint{0.137cm}{1.258cm}}
\pgfpathcurveto{\pgfqpoint{0.173cm}{1.258cm}}{\pgfqpoint{0.207cm}{1.273cm}}{\pgfqpoint{0.233cm}{1.299cm}}
\pgfpathcurveto{\pgfqpoint{0.259cm}{1.324cm}}{\pgfqpoint{0.273cm}{1.359cm}}{\pgfqpoint{0.273cm}{1.395cm}}
\pgfusepath{fill}
\begin{pgfscope}
\pgfsetdash{}{0cm}
\pgfsetlinewidth{0.818mm}
\pgfsetmiterlimit{7.0}
\pgfpathmoveto{\pgfqpoint{0.682cm}{0.671cm}}
\pgfpathlineto{\pgfqpoint{0.679cm}{1.418cm}}
\pgfusepath{stroke}
\end{pgfscope}
\pgfpathmoveto{\pgfqpoint{0.815cm}{1.399cm}}
\pgfpathcurveto{\pgfqpoint{0.815cm}{1.435cm}}{\pgfqpoint{0.801cm}{1.47cm}}{\pgfqpoint{0.775cm}{1.496cm}}
\pgfpathcurveto{\pgfqpoint{0.75cm}{1.521cm}}{\pgfqpoint{0.715cm}{1.536cm}}{\pgfqpoint{0.679cm}{1.536cm}}
\pgfpathcurveto{\pgfqpoint{0.643cm}{1.536cm}}{\pgfqpoint{0.608cm}{1.521cm}}{\pgfqpoint{0.582cm}{1.496cm}}
\pgfpathcurveto{\pgfqpoint{0.557cm}{1.47cm}}{\pgfqpoint{0.542cm}{1.435cm}}{\pgfqpoint{0.542cm}{1.399cm}}
\pgfpathcurveto{\pgfqpoint{0.542cm}{1.363cm}}{\pgfqpoint{0.557cm}{1.328cm}}{\pgfqpoint{0.582cm}{1.302cm}}
\pgfpathcurveto{\pgfqpoint{0.608cm}{1.276cm}}{\pgfqpoint{0.643cm}{1.262cm}}{\pgfqpoint{0.679cm}{1.262cm}}
\pgfpathcurveto{\pgfqpoint{0.715cm}{1.262cm}}{\pgfqpoint{0.75cm}{1.276cm}}{\pgfqpoint{0.775cm}{1.302cm}}
\pgfpathcurveto{\pgfqpoint{0.801cm}{1.328cm}}{\pgfqpoint{0.815cm}{1.363cm}}{\pgfqpoint{0.815cm}{1.399cm}}
\pgfusepath{fill}
\pgfpathmoveto{\pgfqpoint{1.345cm}{1.371cm}}
\pgfpathcurveto{\pgfqpoint{1.345cm}{1.408cm}}{\pgfqpoint{1.331cm}{1.442cm}}{\pgfqpoint{1.305cm}{1.468cm}}
\pgfpathcurveto{\pgfqpoint{1.28cm}{1.494cm}}{\pgfqpoint{1.245cm}{1.508cm}}{\pgfqpoint{1.209cm}{1.508cm}}
\pgfpathcurveto{\pgfqpoint{1.172cm}{1.508cm}}{\pgfqpoint{1.138cm}{1.494cm}}{\pgfqpoint{1.112cm}{1.468cm}}
\pgfpathcurveto{\pgfqpoint{1.087cm}{1.442cm}}{\pgfqpoint{1.072cm}{1.408cm}}{\pgfqpoint{1.072cm}{1.371cm}}
\pgfpathcurveto{\pgfqpoint{1.072cm}{1.335cm}}{\pgfqpoint{1.087cm}{1.3cm}}{\pgfqpoint{1.112cm}{1.274cm}}
\pgfpathcurveto{\pgfqpoint{1.138cm}{1.249cm}}{\pgfqpoint{1.172cm}{1.234cm}}{\pgfqpoint{1.209cm}{1.234cm}}
\pgfpathcurveto{\pgfqpoint{1.245cm}{1.234cm}}{\pgfqpoint{1.28cm}{1.249cm}}{\pgfqpoint{1.305cm}{1.274cm}}
\pgfpathcurveto{\pgfqpoint{1.331cm}{1.3cm}}{\pgfqpoint{1.345cm}{1.335cm}}{\pgfqpoint{1.345cm}{1.371cm}}
\pgfusepath{fill}
\begin{pgfscope}
\pgfsetdash{}{0cm}
\pgfsetlinewidth{0.818mm}
\pgfsetroundcap
\pgfsetmiterlimit{4.0}
\pgfpathmoveto{\pgfqpoint{0.682cm}{0.671cm}}
\pgfpathlineto{\pgfqpoint{0.682cm}{0.042cm}}
\pgfusepath{stroke}
\end{pgfscope}
\end{pgfscope}
\end{pgfscope}
\end{pgfscope}
\end{tikzpicture}}}(\phi+\psi))\|_{\CC^{-\frac{3}{2}+\alpha}(\rho)}+\|6\UU_> X\prec(X^{\!\resizebox{0.6em}{!}{
\begin{tikzpicture}
\pgfpathmoveto{\pgfqpoint{0cm}{-0.035cm}}
\pgfpathlineto{\pgfqpoint{1.376cm}{-0.035cm}}
\pgfpathlineto{\pgfqpoint{1.376cm}{1.552cm}}
\pgfpathlineto{\pgfqpoint{0cm}{1.552cm}}
\pgfpathclose
\pgfusepath{clip}
\begin{pgfscope}
\begin{pgfscope}
\pgfpathmoveto{\pgfqpoint{0cm}{-0.035cm}}
\pgfpathlineto{\pgfqpoint{1.376cm}{-0.035cm}}
\pgfpathlineto{\pgfqpoint{1.376cm}{1.552cm}}
\pgfpathlineto{\pgfqpoint{0cm}{1.552cm}}
\pgfpathclose
\pgfusepath{clip}
\begin{pgfscope}
\begin{pgfscope}
\pgfsetdash{}{0cm}
\pgfsetlinewidth{0.818mm}
\pgfsetroundcap
\pgfsetroundjoin
\pgfsetmiterlimit{7.0}
\definecolor{eps2pgf_color}{gray}{0}\pgfsetstrokecolor{eps2pgf_color}\pgfsetfillcolor{eps2pgf_color}
\pgfpathmoveto{\pgfqpoint{0.117cm}{1.421cm}}
\pgfpathlineto{\pgfqpoint{0.682cm}{0.671cm}}
\pgfpathlineto{\pgfqpoint{1.246cm}{1.421cm}}
\pgfusepath{stroke}
\end{pgfscope}
\definecolor{eps2pgf_color}{gray}{0}\pgfsetstrokecolor{eps2pgf_color}\pgfsetfillcolor{eps2pgf_color}
\pgfpathmoveto{\pgfqpoint{0.273cm}{1.395cm}}
\pgfpathcurveto{\pgfqpoint{0.273cm}{1.432cm}}{\pgfqpoint{0.259cm}{1.467cm}}{\pgfqpoint{0.233cm}{1.492cm}}
\pgfpathcurveto{\pgfqpoint{0.207cm}{1.518cm}}{\pgfqpoint{0.173cm}{1.532cm}}{\pgfqpoint{0.137cm}{1.532cm}}
\pgfpathcurveto{\pgfqpoint{0.1cm}{1.532cm}}{\pgfqpoint{0.066cm}{1.518cm}}{\pgfqpoint{0.04cm}{1.492cm}}
\pgfpathcurveto{\pgfqpoint{0.014cm}{1.467cm}}{\pgfqpoint{0cm}{1.432cm}}{\pgfqpoint{0cm}{1.395cm}}
\pgfpathcurveto{\pgfqpoint{0cm}{1.359cm}}{\pgfqpoint{0.014cm}{1.324cm}}{\pgfqpoint{0.04cm}{1.299cm}}
\pgfpathcurveto{\pgfqpoint{0.066cm}{1.273cm}}{\pgfqpoint{0.1cm}{1.258cm}}{\pgfqpoint{0.137cm}{1.258cm}}
\pgfpathcurveto{\pgfqpoint{0.173cm}{1.258cm}}{\pgfqpoint{0.207cm}{1.273cm}}{\pgfqpoint{0.233cm}{1.299cm}}
\pgfpathcurveto{\pgfqpoint{0.259cm}{1.324cm}}{\pgfqpoint{0.273cm}{1.359cm}}{\pgfqpoint{0.273cm}{1.395cm}}
\pgfusepath{fill}
\begin{pgfscope}
\pgfsetdash{}{0cm}
\pgfsetlinewidth{0.818mm}
\pgfsetmiterlimit{7.0}
\pgfpathmoveto{\pgfqpoint{0.682cm}{0.671cm}}
\pgfpathlineto{\pgfqpoint{0.679cm}{1.418cm}}
\pgfusepath{stroke}
\end{pgfscope}
\pgfpathmoveto{\pgfqpoint{0.815cm}{1.399cm}}
\pgfpathcurveto{\pgfqpoint{0.815cm}{1.435cm}}{\pgfqpoint{0.801cm}{1.47cm}}{\pgfqpoint{0.775cm}{1.496cm}}
\pgfpathcurveto{\pgfqpoint{0.75cm}{1.521cm}}{\pgfqpoint{0.715cm}{1.536cm}}{\pgfqpoint{0.679cm}{1.536cm}}
\pgfpathcurveto{\pgfqpoint{0.643cm}{1.536cm}}{\pgfqpoint{0.608cm}{1.521cm}}{\pgfqpoint{0.582cm}{1.496cm}}
\pgfpathcurveto{\pgfqpoint{0.557cm}{1.47cm}}{\pgfqpoint{0.542cm}{1.435cm}}{\pgfqpoint{0.542cm}{1.399cm}}
\pgfpathcurveto{\pgfqpoint{0.542cm}{1.363cm}}{\pgfqpoint{0.557cm}{1.328cm}}{\pgfqpoint{0.582cm}{1.302cm}}
\pgfpathcurveto{\pgfqpoint{0.608cm}{1.276cm}}{\pgfqpoint{0.643cm}{1.262cm}}{\pgfqpoint{0.679cm}{1.262cm}}
\pgfpathcurveto{\pgfqpoint{0.715cm}{1.262cm}}{\pgfqpoint{0.75cm}{1.276cm}}{\pgfqpoint{0.775cm}{1.302cm}}
\pgfpathcurveto{\pgfqpoint{0.801cm}{1.328cm}}{\pgfqpoint{0.815cm}{1.363cm}}{\pgfqpoint{0.815cm}{1.399cm}}
\pgfusepath{fill}
\pgfpathmoveto{\pgfqpoint{1.345cm}{1.371cm}}
\pgfpathcurveto{\pgfqpoint{1.345cm}{1.408cm}}{\pgfqpoint{1.331cm}{1.442cm}}{\pgfqpoint{1.305cm}{1.468cm}}
\pgfpathcurveto{\pgfqpoint{1.28cm}{1.494cm}}{\pgfqpoint{1.245cm}{1.508cm}}{\pgfqpoint{1.209cm}{1.508cm}}
\pgfpathcurveto{\pgfqpoint{1.172cm}{1.508cm}}{\pgfqpoint{1.138cm}{1.494cm}}{\pgfqpoint{1.112cm}{1.468cm}}
\pgfpathcurveto{\pgfqpoint{1.087cm}{1.442cm}}{\pgfqpoint{1.072cm}{1.408cm}}{\pgfqpoint{1.072cm}{1.371cm}}
\pgfpathcurveto{\pgfqpoint{1.072cm}{1.335cm}}{\pgfqpoint{1.087cm}{1.3cm}}{\pgfqpoint{1.112cm}{1.274cm}}
\pgfpathcurveto{\pgfqpoint{1.138cm}{1.249cm}}{\pgfqpoint{1.172cm}{1.234cm}}{\pgfqpoint{1.209cm}{1.234cm}}
\pgfpathcurveto{\pgfqpoint{1.245cm}{1.234cm}}{\pgfqpoint{1.28cm}{1.249cm}}{\pgfqpoint{1.305cm}{1.274cm}}
\pgfpathcurveto{\pgfqpoint{1.331cm}{1.3cm}}{\pgfqpoint{1.345cm}{1.335cm}}{\pgfqpoint{1.345cm}{1.371cm}}
\pgfusepath{fill}
\begin{pgfscope}
\pgfsetdash{}{0cm}
\pgfsetlinewidth{0.818mm}
\pgfsetroundcap
\pgfsetmiterlimit{4.0}
\pgfpathmoveto{\pgfqpoint{0.682cm}{0.671cm}}
\pgfpathlineto{\pgfqpoint{0.682cm}{0.042cm}}
\pgfusepath{stroke}
\end{pgfscope}
\end{pgfscope}
\end{pgfscope}
\end{pgfscope}
\end{tikzpicture}}}(\phi+\psi))\|_{\CC^{-\frac{3}{2}+\alpha}(\rho)}\\
&\quad\lesssim  \|\phi+\psi\|_{L^\infty(\rho)}\|\UU_>X\|_{\CC^{-\frac{3}{2}+\alpha}(\rho^{-\alpha})}\\
&\quad\lesssim 2^{-(1-\alpha-\kappa)\frac{2}{3}K}\|\phi+\psi\|_{L^\infty(\rho)}\lesssim 1+\|\psi\|_{L^\infty(\rho)}^{\varepsilon},
\end{align*}
\begin{align*}
\|6(\phi+\psi)\prec\UU_{>}X^{\!\resizebox{!}{.8em}{
\begin{tikzpicture}
\pgfpathmoveto{\pgfqpoint{0cm}{-0.035cm}}
\pgfpathlineto{\pgfqpoint{1.976cm}{-0.035cm}}
\pgfpathlineto{\pgfqpoint{1.976cm}{1.94cm}}
\pgfpathlineto{\pgfqpoint{0cm}{1.94cm}}
\pgfpathclose
\pgfusepath{clip}
\begin{pgfscope}
\begin{pgfscope}
\pgfpathmoveto{\pgfqpoint{0cm}{-0.035cm}}
\pgfpathlineto{\pgfqpoint{1.976cm}{-0.035cm}}
\pgfpathlineto{\pgfqpoint{1.976cm}{1.94cm}}
\pgfpathlineto{\pgfqpoint{0cm}{1.94cm}}
\pgfpathclose
\pgfusepath{clip}
\begin{pgfscope}
\begin{pgfscope}
\pgfsetdash{}{0cm}
\pgfsetlinewidth{0.818mm}
\pgfsetroundcap
\pgfsetroundjoin
\pgfsetmiterlimit{7.0}
\definecolor{eps2pgf_color}{gray}{0}\pgfsetstrokecolor{eps2pgf_color}\pgfsetfillcolor{eps2pgf_color}
\pgfpathmoveto{\pgfqpoint{0.117cm}{1.815cm}}
\pgfpathlineto{\pgfqpoint{0.682cm}{1.065cm}}
\pgfpathlineto{\pgfqpoint{1.246cm}{1.815cm}}
\pgfusepath{stroke}
\end{pgfscope}
\definecolor{eps2pgf_color}{gray}{0}\pgfsetstrokecolor{eps2pgf_color}\pgfsetfillcolor{eps2pgf_color}
\pgfpathmoveto{\pgfqpoint{0.273cm}{1.789cm}}
\pgfpathcurveto{\pgfqpoint{0.273cm}{1.825cm}}{\pgfqpoint{0.259cm}{1.86cm}}{\pgfqpoint{0.233cm}{1.886cm}}
\pgfpathcurveto{\pgfqpoint{0.207cm}{1.912cm}}{\pgfqpoint{0.173cm}{1.926cm}}{\pgfqpoint{0.137cm}{1.926cm}}
\pgfpathcurveto{\pgfqpoint{0.1cm}{1.926cm}}{\pgfqpoint{0.066cm}{1.912cm}}{\pgfqpoint{0.04cm}{1.886cm}}
\pgfpathcurveto{\pgfqpoint{0.014cm}{1.86cm}}{\pgfqpoint{0cm}{1.825cm}}{\pgfqpoint{0cm}{1.789cm}}
\pgfpathcurveto{\pgfqpoint{0cm}{1.753cm}}{\pgfqpoint{0.014cm}{1.718cm}}{\pgfqpoint{0.04cm}{1.692cm}}
\pgfpathcurveto{\pgfqpoint{0.066cm}{1.667cm}}{\pgfqpoint{0.1cm}{1.652cm}}{\pgfqpoint{0.137cm}{1.652cm}}
\pgfpathcurveto{\pgfqpoint{0.173cm}{1.652cm}}{\pgfqpoint{0.207cm}{1.667cm}}{\pgfqpoint{0.233cm}{1.692cm}}
\pgfpathcurveto{\pgfqpoint{0.259cm}{1.718cm}}{\pgfqpoint{0.273cm}{1.753cm}}{\pgfqpoint{0.273cm}{1.789cm}}
\pgfusepath{fill}
\begin{pgfscope}
\pgfsetdash{}{0cm}
\pgfsetlinewidth{0.818mm}
\pgfsetmiterlimit{7.0}
\pgfpathmoveto{\pgfqpoint{0.682cm}{1.065cm}}
\pgfpathlineto{\pgfqpoint{0.679cm}{1.812cm}}
\pgfusepath{stroke}
\end{pgfscope}
\pgfpathmoveto{\pgfqpoint{0.815cm}{1.793cm}}
\pgfpathcurveto{\pgfqpoint{0.815cm}{1.829cm}}{\pgfqpoint{0.801cm}{1.864cm}}{\pgfqpoint{0.775cm}{1.89cm}}
\pgfpathcurveto{\pgfqpoint{0.75cm}{1.915cm}}{\pgfqpoint{0.715cm}{1.93cm}}{\pgfqpoint{0.679cm}{1.93cm}}
\pgfpathcurveto{\pgfqpoint{0.643cm}{1.93cm}}{\pgfqpoint{0.608cm}{1.915cm}}{\pgfqpoint{0.582cm}{1.89cm}}
\pgfpathcurveto{\pgfqpoint{0.557cm}{1.864cm}}{\pgfqpoint{0.542cm}{1.829cm}}{\pgfqpoint{0.542cm}{1.793cm}}
\pgfpathcurveto{\pgfqpoint{0.542cm}{1.756cm}}{\pgfqpoint{0.557cm}{1.722cm}}{\pgfqpoint{0.582cm}{1.696cm}}
\pgfpathcurveto{\pgfqpoint{0.608cm}{1.67cm}}{\pgfqpoint{0.643cm}{1.656cm}}{\pgfqpoint{0.679cm}{1.656cm}}
\pgfpathcurveto{\pgfqpoint{0.715cm}{1.656cm}}{\pgfqpoint{0.75cm}{1.67cm}}{\pgfqpoint{0.775cm}{1.696cm}}
\pgfpathcurveto{\pgfqpoint{0.801cm}{1.722cm}}{\pgfqpoint{0.815cm}{1.756cm}}{\pgfqpoint{0.815cm}{1.793cm}}
\pgfusepath{fill}
\pgfpathmoveto{\pgfqpoint{1.345cm}{1.765cm}}
\pgfpathcurveto{\pgfqpoint{1.345cm}{1.801cm}}{\pgfqpoint{1.331cm}{1.836cm}}{\pgfqpoint{1.305cm}{1.862cm}}
\pgfpathcurveto{\pgfqpoint{1.28cm}{1.887cm}}{\pgfqpoint{1.245cm}{1.902cm}}{\pgfqpoint{1.209cm}{1.902cm}}
\pgfpathcurveto{\pgfqpoint{1.172cm}{1.902cm}}{\pgfqpoint{1.138cm}{1.887cm}}{\pgfqpoint{1.112cm}{1.862cm}}
\pgfpathcurveto{\pgfqpoint{1.087cm}{1.836cm}}{\pgfqpoint{1.072cm}{1.801cm}}{\pgfqpoint{1.072cm}{1.765cm}}
\pgfpathcurveto{\pgfqpoint{1.072cm}{1.728cm}}{\pgfqpoint{1.087cm}{1.694cm}}{\pgfqpoint{1.112cm}{1.668cm}}
\pgfpathcurveto{\pgfqpoint{1.138cm}{1.642cm}}{\pgfqpoint{1.172cm}{1.628cm}}{\pgfqpoint{1.209cm}{1.628cm}}
\pgfpathcurveto{\pgfqpoint{1.245cm}{1.628cm}}{\pgfqpoint{1.28cm}{1.642cm}}{\pgfqpoint{1.305cm}{1.668cm}}
\pgfpathcurveto{\pgfqpoint{1.331cm}{1.694cm}}{\pgfqpoint{1.345cm}{1.728cm}}{\pgfqpoint{1.345cm}{1.765cm}}
\pgfusepath{fill}
\begin{pgfscope}
\pgfsetdash{}{0cm}
\pgfsetlinewidth{0.818mm}
\pgfsetroundcap
\pgfsetroundjoin
\pgfsetmiterlimit{7.0}
\pgfpathmoveto{\pgfqpoint{0.682cm}{1.065cm}}
\pgfpathlineto{\pgfqpoint{1.246cm}{0.315cm}}
\pgfpathlineto{\pgfqpoint{1.811cm}{1.065cm}}
\pgfusepath{stroke}
\end{pgfscope}
\pgfpathmoveto{\pgfqpoint{1.948cm}{1.065cm}}
\pgfpathcurveto{\pgfqpoint{1.948cm}{1.101cm}}{\pgfqpoint{1.933cm}{1.136cm}}{\pgfqpoint{1.907cm}{1.162cm}}
\pgfpathcurveto{\pgfqpoint{1.882cm}{1.187cm}}{\pgfqpoint{1.847cm}{1.202cm}}{\pgfqpoint{1.811cm}{1.202cm}}
\pgfpathcurveto{\pgfqpoint{1.775cm}{1.202cm}}{\pgfqpoint{1.74cm}{1.187cm}}{\pgfqpoint{1.714cm}{1.162cm}}
\pgfpathcurveto{\pgfqpoint{1.689cm}{1.136cm}}{\pgfqpoint{1.674cm}{1.101cm}}{\pgfqpoint{1.674cm}{1.065cm}}
\pgfpathcurveto{\pgfqpoint{1.674cm}{1.029cm}}{\pgfqpoint{1.689cm}{0.994cm}}{\pgfqpoint{1.714cm}{0.968cm}}
\pgfpathcurveto{\pgfqpoint{1.74cm}{0.942cm}}{\pgfqpoint{1.775cm}{0.928cm}}{\pgfqpoint{1.811cm}{0.928cm}}
\pgfpathcurveto{\pgfqpoint{1.847cm}{0.928cm}}{\pgfqpoint{1.882cm}{0.942cm}}{\pgfqpoint{1.907cm}{0.968cm}}
\pgfpathcurveto{\pgfqpoint{1.933cm}{0.994cm}}{\pgfqpoint{1.948cm}{1.029cm}}{\pgfqpoint{1.948cm}{1.065cm}}
\pgfusepath{fill}
\begin{pgfscope}
\pgfsetdash{}{0cm}
\pgfsetlinewidth{0.818mm}
\pgfsetmiterlimit{4.0}
\pgfpathmoveto{\pgfqpoint{1.383cm}{0.178cm}}
\pgfpathcurveto{\pgfqpoint{1.383cm}{0.214cm}}{\pgfqpoint{1.369cm}{0.249cm}}{\pgfqpoint{1.343cm}{0.275cm}}
\pgfpathcurveto{\pgfqpoint{1.317cm}{0.3cm}}{\pgfqpoint{1.283cm}{0.315cm}}{\pgfqpoint{1.246cm}{0.315cm}}
\pgfpathcurveto{\pgfqpoint{1.21cm}{0.315cm}}{\pgfqpoint{1.175cm}{0.3cm}}{\pgfqpoint{1.15cm}{0.275cm}}
\pgfpathcurveto{\pgfqpoint{1.124cm}{0.249cm}}{\pgfqpoint{1.11cm}{0.214cm}}{\pgfqpoint{1.11cm}{0.178cm}}
\pgfpathcurveto{\pgfqpoint{1.11cm}{0.141cm}}{\pgfqpoint{1.124cm}{0.107cm}}{\pgfqpoint{1.15cm}{0.081cm}}
\pgfpathcurveto{\pgfqpoint{1.175cm}{0.055cm}}{\pgfqpoint{1.21cm}{0.041cm}}{\pgfqpoint{1.246cm}{0.041cm}}
\pgfpathcurveto{\pgfqpoint{1.283cm}{0.041cm}}{\pgfqpoint{1.317cm}{0.055cm}}{\pgfqpoint{1.343cm}{0.081cm}}
\pgfpathcurveto{\pgfqpoint{1.369cm}{0.107cm}}{\pgfqpoint{1.383cm}{0.141cm}}{\pgfqpoint{1.383cm}{0.178cm}}
\pgfusepath{stroke}
\end{pgfscope}
\end{pgfscope}
\end{pgfscope}
\end{pgfscope}
\end{tikzpicture}}}\|_{\CC^{-\frac{3}{2}+\alpha}(\rho)}
&\lesssim \|\phi+\psi\|_{L^\infty(\rho)}\|\UU_>X^{\!\resizebox{!}{.8em}{
\begin{tikzpicture}
\pgfpathmoveto{\pgfqpoint{0cm}{-0.035cm}}
\pgfpathlineto{\pgfqpoint{1.976cm}{-0.035cm}}
\pgfpathlineto{\pgfqpoint{1.976cm}{1.94cm}}
\pgfpathlineto{\pgfqpoint{0cm}{1.94cm}}
\pgfpathclose
\pgfusepath{clip}
\begin{pgfscope}
\begin{pgfscope}
\pgfpathmoveto{\pgfqpoint{0cm}{-0.035cm}}
\pgfpathlineto{\pgfqpoint{1.976cm}{-0.035cm}}
\pgfpathlineto{\pgfqpoint{1.976cm}{1.94cm}}
\pgfpathlineto{\pgfqpoint{0cm}{1.94cm}}
\pgfpathclose
\pgfusepath{clip}
\begin{pgfscope}
\begin{pgfscope}
\pgfsetdash{}{0cm}
\pgfsetlinewidth{0.818mm}
\pgfsetroundcap
\pgfsetroundjoin
\pgfsetmiterlimit{7.0}
\definecolor{eps2pgf_color}{gray}{0}\pgfsetstrokecolor{eps2pgf_color}\pgfsetfillcolor{eps2pgf_color}
\pgfpathmoveto{\pgfqpoint{0.117cm}{1.815cm}}
\pgfpathlineto{\pgfqpoint{0.682cm}{1.065cm}}
\pgfpathlineto{\pgfqpoint{1.246cm}{1.815cm}}
\pgfusepath{stroke}
\end{pgfscope}
\definecolor{eps2pgf_color}{gray}{0}\pgfsetstrokecolor{eps2pgf_color}\pgfsetfillcolor{eps2pgf_color}
\pgfpathmoveto{\pgfqpoint{0.273cm}{1.789cm}}
\pgfpathcurveto{\pgfqpoint{0.273cm}{1.825cm}}{\pgfqpoint{0.259cm}{1.86cm}}{\pgfqpoint{0.233cm}{1.886cm}}
\pgfpathcurveto{\pgfqpoint{0.207cm}{1.912cm}}{\pgfqpoint{0.173cm}{1.926cm}}{\pgfqpoint{0.137cm}{1.926cm}}
\pgfpathcurveto{\pgfqpoint{0.1cm}{1.926cm}}{\pgfqpoint{0.066cm}{1.912cm}}{\pgfqpoint{0.04cm}{1.886cm}}
\pgfpathcurveto{\pgfqpoint{0.014cm}{1.86cm}}{\pgfqpoint{0cm}{1.825cm}}{\pgfqpoint{0cm}{1.789cm}}
\pgfpathcurveto{\pgfqpoint{0cm}{1.753cm}}{\pgfqpoint{0.014cm}{1.718cm}}{\pgfqpoint{0.04cm}{1.692cm}}
\pgfpathcurveto{\pgfqpoint{0.066cm}{1.667cm}}{\pgfqpoint{0.1cm}{1.652cm}}{\pgfqpoint{0.137cm}{1.652cm}}
\pgfpathcurveto{\pgfqpoint{0.173cm}{1.652cm}}{\pgfqpoint{0.207cm}{1.667cm}}{\pgfqpoint{0.233cm}{1.692cm}}
\pgfpathcurveto{\pgfqpoint{0.259cm}{1.718cm}}{\pgfqpoint{0.273cm}{1.753cm}}{\pgfqpoint{0.273cm}{1.789cm}}
\pgfusepath{fill}
\begin{pgfscope}
\pgfsetdash{}{0cm}
\pgfsetlinewidth{0.818mm}
\pgfsetmiterlimit{7.0}
\pgfpathmoveto{\pgfqpoint{0.682cm}{1.065cm}}
\pgfpathlineto{\pgfqpoint{0.679cm}{1.812cm}}
\pgfusepath{stroke}
\end{pgfscope}
\pgfpathmoveto{\pgfqpoint{0.815cm}{1.793cm}}
\pgfpathcurveto{\pgfqpoint{0.815cm}{1.829cm}}{\pgfqpoint{0.801cm}{1.864cm}}{\pgfqpoint{0.775cm}{1.89cm}}
\pgfpathcurveto{\pgfqpoint{0.75cm}{1.915cm}}{\pgfqpoint{0.715cm}{1.93cm}}{\pgfqpoint{0.679cm}{1.93cm}}
\pgfpathcurveto{\pgfqpoint{0.643cm}{1.93cm}}{\pgfqpoint{0.608cm}{1.915cm}}{\pgfqpoint{0.582cm}{1.89cm}}
\pgfpathcurveto{\pgfqpoint{0.557cm}{1.864cm}}{\pgfqpoint{0.542cm}{1.829cm}}{\pgfqpoint{0.542cm}{1.793cm}}
\pgfpathcurveto{\pgfqpoint{0.542cm}{1.756cm}}{\pgfqpoint{0.557cm}{1.722cm}}{\pgfqpoint{0.582cm}{1.696cm}}
\pgfpathcurveto{\pgfqpoint{0.608cm}{1.67cm}}{\pgfqpoint{0.643cm}{1.656cm}}{\pgfqpoint{0.679cm}{1.656cm}}
\pgfpathcurveto{\pgfqpoint{0.715cm}{1.656cm}}{\pgfqpoint{0.75cm}{1.67cm}}{\pgfqpoint{0.775cm}{1.696cm}}
\pgfpathcurveto{\pgfqpoint{0.801cm}{1.722cm}}{\pgfqpoint{0.815cm}{1.756cm}}{\pgfqpoint{0.815cm}{1.793cm}}
\pgfusepath{fill}
\pgfpathmoveto{\pgfqpoint{1.345cm}{1.765cm}}
\pgfpathcurveto{\pgfqpoint{1.345cm}{1.801cm}}{\pgfqpoint{1.331cm}{1.836cm}}{\pgfqpoint{1.305cm}{1.862cm}}
\pgfpathcurveto{\pgfqpoint{1.28cm}{1.887cm}}{\pgfqpoint{1.245cm}{1.902cm}}{\pgfqpoint{1.209cm}{1.902cm}}
\pgfpathcurveto{\pgfqpoint{1.172cm}{1.902cm}}{\pgfqpoint{1.138cm}{1.887cm}}{\pgfqpoint{1.112cm}{1.862cm}}
\pgfpathcurveto{\pgfqpoint{1.087cm}{1.836cm}}{\pgfqpoint{1.072cm}{1.801cm}}{\pgfqpoint{1.072cm}{1.765cm}}
\pgfpathcurveto{\pgfqpoint{1.072cm}{1.728cm}}{\pgfqpoint{1.087cm}{1.694cm}}{\pgfqpoint{1.112cm}{1.668cm}}
\pgfpathcurveto{\pgfqpoint{1.138cm}{1.642cm}}{\pgfqpoint{1.172cm}{1.628cm}}{\pgfqpoint{1.209cm}{1.628cm}}
\pgfpathcurveto{\pgfqpoint{1.245cm}{1.628cm}}{\pgfqpoint{1.28cm}{1.642cm}}{\pgfqpoint{1.305cm}{1.668cm}}
\pgfpathcurveto{\pgfqpoint{1.331cm}{1.694cm}}{\pgfqpoint{1.345cm}{1.728cm}}{\pgfqpoint{1.345cm}{1.765cm}}
\pgfusepath{fill}
\begin{pgfscope}
\pgfsetdash{}{0cm}
\pgfsetlinewidth{0.818mm}
\pgfsetroundcap
\pgfsetroundjoin
\pgfsetmiterlimit{7.0}
\pgfpathmoveto{\pgfqpoint{0.682cm}{1.065cm}}
\pgfpathlineto{\pgfqpoint{1.246cm}{0.315cm}}
\pgfpathlineto{\pgfqpoint{1.811cm}{1.065cm}}
\pgfusepath{stroke}
\end{pgfscope}
\pgfpathmoveto{\pgfqpoint{1.948cm}{1.065cm}}
\pgfpathcurveto{\pgfqpoint{1.948cm}{1.101cm}}{\pgfqpoint{1.933cm}{1.136cm}}{\pgfqpoint{1.907cm}{1.162cm}}
\pgfpathcurveto{\pgfqpoint{1.882cm}{1.187cm}}{\pgfqpoint{1.847cm}{1.202cm}}{\pgfqpoint{1.811cm}{1.202cm}}
\pgfpathcurveto{\pgfqpoint{1.775cm}{1.202cm}}{\pgfqpoint{1.74cm}{1.187cm}}{\pgfqpoint{1.714cm}{1.162cm}}
\pgfpathcurveto{\pgfqpoint{1.689cm}{1.136cm}}{\pgfqpoint{1.674cm}{1.101cm}}{\pgfqpoint{1.674cm}{1.065cm}}
\pgfpathcurveto{\pgfqpoint{1.674cm}{1.029cm}}{\pgfqpoint{1.689cm}{0.994cm}}{\pgfqpoint{1.714cm}{0.968cm}}
\pgfpathcurveto{\pgfqpoint{1.74cm}{0.942cm}}{\pgfqpoint{1.775cm}{0.928cm}}{\pgfqpoint{1.811cm}{0.928cm}}
\pgfpathcurveto{\pgfqpoint{1.847cm}{0.928cm}}{\pgfqpoint{1.882cm}{0.942cm}}{\pgfqpoint{1.907cm}{0.968cm}}
\pgfpathcurveto{\pgfqpoint{1.933cm}{0.994cm}}{\pgfqpoint{1.948cm}{1.029cm}}{\pgfqpoint{1.948cm}{1.065cm}}
\pgfusepath{fill}
\begin{pgfscope}
\pgfsetdash{}{0cm}
\pgfsetlinewidth{0.818mm}
\pgfsetmiterlimit{4.0}
\pgfpathmoveto{\pgfqpoint{1.383cm}{0.178cm}}
\pgfpathcurveto{\pgfqpoint{1.383cm}{0.214cm}}{\pgfqpoint{1.369cm}{0.249cm}}{\pgfqpoint{1.343cm}{0.275cm}}
\pgfpathcurveto{\pgfqpoint{1.317cm}{0.3cm}}{\pgfqpoint{1.283cm}{0.315cm}}{\pgfqpoint{1.246cm}{0.315cm}}
\pgfpathcurveto{\pgfqpoint{1.21cm}{0.315cm}}{\pgfqpoint{1.175cm}{0.3cm}}{\pgfqpoint{1.15cm}{0.275cm}}
\pgfpathcurveto{\pgfqpoint{1.124cm}{0.249cm}}{\pgfqpoint{1.11cm}{0.214cm}}{\pgfqpoint{1.11cm}{0.178cm}}
\pgfpathcurveto{\pgfqpoint{1.11cm}{0.141cm}}{\pgfqpoint{1.124cm}{0.107cm}}{\pgfqpoint{1.15cm}{0.081cm}}
\pgfpathcurveto{\pgfqpoint{1.175cm}{0.055cm}}{\pgfqpoint{1.21cm}{0.041cm}}{\pgfqpoint{1.246cm}{0.041cm}}
\pgfpathcurveto{\pgfqpoint{1.283cm}{0.041cm}}{\pgfqpoint{1.317cm}{0.055cm}}{\pgfqpoint{1.343cm}{0.081cm}}
\pgfpathcurveto{\pgfqpoint{1.369cm}{0.107cm}}{\pgfqpoint{1.383cm}{0.141cm}}{\pgfqpoint{1.383cm}{0.178cm}}
\pgfusepath{stroke}
\end{pgfscope}
\end{pgfscope}
\end{pgfscope}
\end{pgfscope}
\end{tikzpicture}}}\|_{\CC^{-\frac{3}{2}+\alpha}}\\
&\lesssim 2^{-(\frac{3}{2}-\alpha-\kappa)K/2}\|\phi+\psi\|_{L^\infty(\rho)}\lesssim 1+\|\psi\|_{L^\infty(\rho)}^{\varepsilon}.
\end{align*}
For the last term, we note that a higher power of $\rho$ is necessary and estimate
\begin{align*}
\|3\UU_> X\succ(\phi+\psi)^2\|_{\CC^{-\frac{3}{2}+\alpha}(\rho^{\frac{3}{2}+\alpha})}
&\lesssim \|\phi+\psi\|^2_{L^\infty(\rho)}\|\UU_>X\|_{\CC^{-\frac{3}{2}+\alpha}(\rho^{-\frac{1}{2}+\alpha})}\\
&\lesssim 2^{-(1-\alpha-\kappa)\frac{4}{3}K}\|\phi+\psi\|_{L^\infty(\rho)}^2\lesssim 1+\|\psi\|_{L^\infty(\rho)}^{\varepsilon},
\end{align*}

To summarize, collecting all the above estimates and using the Schauder estimates we  deduce that 
\begin{align}\label{eq:11}
& \|\phi\|_{\CC^{\frac{1}{2}+\alpha}(\rho^{\frac{3}{2}+\alpha})}\lesssim\|\Phi\|_{\CC^{-\frac{3}{2}+\alpha}(\rho^{\frac{3}{2}+\alpha})}\lesssim 1+\|\psi\|_{L^\infty(\rho)}^{\varepsilon}.
\end{align}

\subsection{Bound for $\vartheta$ in $\CC^{1+\alpha}(\rho^{2+\alpha})$}
\label{ssec:theta}

As the next step, we derive a bound for $\vartheta$ in $\CC^{1+\alpha}(\rho^{2+\alpha})$ which will be needed in the sequel in order to control $\psi$. In view of the paracontrolled ansatz \eqref{eq:th}, equation \eqref{eq:rhs45a} as well as the decomposition \eqref{eq:two}, we observe that the most irregular part of $\Phi$, namely the two magenta terms coming from \eqref{eq:2}, cancel out, and additionally the blue term coming from \eqref{eq:2} and a commutator appear. More precisely,  $\vartheta$ solves
$$
\Q\vartheta + \Theta =0,
$$
with
$$
\Theta=\Phi +\rmm{3\llbracket X^2 \rrbracket\succX^{\!\resizebox{0.6em}{!}{
\begin{tikzpicture}
\pgfpathmoveto{\pgfqpoint{0cm}{-0.035cm}}
\pgfpathlineto{\pgfqpoint{1.376cm}{-0.035cm}}
\pgfpathlineto{\pgfqpoint{1.376cm}{1.552cm}}
\pgfpathlineto{\pgfqpoint{0cm}{1.552cm}}
\pgfpathclose
\pgfusepath{clip}
\begin{pgfscope}
\begin{pgfscope}
\pgfpathmoveto{\pgfqpoint{0cm}{-0.035cm}}
\pgfpathlineto{\pgfqpoint{1.376cm}{-0.035cm}}
\pgfpathlineto{\pgfqpoint{1.376cm}{1.552cm}}
\pgfpathlineto{\pgfqpoint{0cm}{1.552cm}}
\pgfpathclose
\pgfusepath{clip}
\begin{pgfscope}
\begin{pgfscope}
\pgfsetdash{}{0cm}
\pgfsetlinewidth{0.818mm}
\pgfsetroundcap
\pgfsetroundjoin
\pgfsetmiterlimit{7.0}
\definecolor{eps2pgf_color}{gray}{0}\pgfsetstrokecolor{eps2pgf_color}\pgfsetfillcolor{eps2pgf_color}
\pgfpathmoveto{\pgfqpoint{0.117cm}{1.421cm}}
\pgfpathlineto{\pgfqpoint{0.682cm}{0.671cm}}
\pgfpathlineto{\pgfqpoint{1.246cm}{1.421cm}}
\pgfusepath{stroke}
\end{pgfscope}
\definecolor{eps2pgf_color}{gray}{0}\pgfsetstrokecolor{eps2pgf_color}\pgfsetfillcolor{eps2pgf_color}
\pgfpathmoveto{\pgfqpoint{0.273cm}{1.395cm}}
\pgfpathcurveto{\pgfqpoint{0.273cm}{1.432cm}}{\pgfqpoint{0.259cm}{1.467cm}}{\pgfqpoint{0.233cm}{1.492cm}}
\pgfpathcurveto{\pgfqpoint{0.207cm}{1.518cm}}{\pgfqpoint{0.173cm}{1.532cm}}{\pgfqpoint{0.137cm}{1.532cm}}
\pgfpathcurveto{\pgfqpoint{0.1cm}{1.532cm}}{\pgfqpoint{0.066cm}{1.518cm}}{\pgfqpoint{0.04cm}{1.492cm}}
\pgfpathcurveto{\pgfqpoint{0.014cm}{1.467cm}}{\pgfqpoint{0cm}{1.432cm}}{\pgfqpoint{0cm}{1.395cm}}
\pgfpathcurveto{\pgfqpoint{0cm}{1.359cm}}{\pgfqpoint{0.014cm}{1.324cm}}{\pgfqpoint{0.04cm}{1.299cm}}
\pgfpathcurveto{\pgfqpoint{0.066cm}{1.273cm}}{\pgfqpoint{0.1cm}{1.258cm}}{\pgfqpoint{0.137cm}{1.258cm}}
\pgfpathcurveto{\pgfqpoint{0.173cm}{1.258cm}}{\pgfqpoint{0.207cm}{1.273cm}}{\pgfqpoint{0.233cm}{1.299cm}}
\pgfpathcurveto{\pgfqpoint{0.259cm}{1.324cm}}{\pgfqpoint{0.273cm}{1.359cm}}{\pgfqpoint{0.273cm}{1.395cm}}
\pgfusepath{fill}
\begin{pgfscope}
\pgfsetdash{}{0cm}
\pgfsetlinewidth{0.818mm}
\pgfsetmiterlimit{7.0}
\pgfpathmoveto{\pgfqpoint{0.682cm}{0.671cm}}
\pgfpathlineto{\pgfqpoint{0.679cm}{1.418cm}}
\pgfusepath{stroke}
\end{pgfscope}
\pgfpathmoveto{\pgfqpoint{0.815cm}{1.399cm}}
\pgfpathcurveto{\pgfqpoint{0.815cm}{1.435cm}}{\pgfqpoint{0.801cm}{1.47cm}}{\pgfqpoint{0.775cm}{1.496cm}}
\pgfpathcurveto{\pgfqpoint{0.75cm}{1.521cm}}{\pgfqpoint{0.715cm}{1.536cm}}{\pgfqpoint{0.679cm}{1.536cm}}
\pgfpathcurveto{\pgfqpoint{0.643cm}{1.536cm}}{\pgfqpoint{0.608cm}{1.521cm}}{\pgfqpoint{0.582cm}{1.496cm}}
\pgfpathcurveto{\pgfqpoint{0.557cm}{1.47cm}}{\pgfqpoint{0.542cm}{1.435cm}}{\pgfqpoint{0.542cm}{1.399cm}}
\pgfpathcurveto{\pgfqpoint{0.542cm}{1.363cm}}{\pgfqpoint{0.557cm}{1.328cm}}{\pgfqpoint{0.582cm}{1.302cm}}
\pgfpathcurveto{\pgfqpoint{0.608cm}{1.276cm}}{\pgfqpoint{0.643cm}{1.262cm}}{\pgfqpoint{0.679cm}{1.262cm}}
\pgfpathcurveto{\pgfqpoint{0.715cm}{1.262cm}}{\pgfqpoint{0.75cm}{1.276cm}}{\pgfqpoint{0.775cm}{1.302cm}}
\pgfpathcurveto{\pgfqpoint{0.801cm}{1.328cm}}{\pgfqpoint{0.815cm}{1.363cm}}{\pgfqpoint{0.815cm}{1.399cm}}
\pgfusepath{fill}
\pgfpathmoveto{\pgfqpoint{1.345cm}{1.371cm}}
\pgfpathcurveto{\pgfqpoint{1.345cm}{1.408cm}}{\pgfqpoint{1.331cm}{1.442cm}}{\pgfqpoint{1.305cm}{1.468cm}}
\pgfpathcurveto{\pgfqpoint{1.28cm}{1.494cm}}{\pgfqpoint{1.245cm}{1.508cm}}{\pgfqpoint{1.209cm}{1.508cm}}
\pgfpathcurveto{\pgfqpoint{1.172cm}{1.508cm}}{\pgfqpoint{1.138cm}{1.494cm}}{\pgfqpoint{1.112cm}{1.468cm}}
\pgfpathcurveto{\pgfqpoint{1.087cm}{1.442cm}}{\pgfqpoint{1.072cm}{1.408cm}}{\pgfqpoint{1.072cm}{1.371cm}}
\pgfpathcurveto{\pgfqpoint{1.072cm}{1.335cm}}{\pgfqpoint{1.087cm}{1.3cm}}{\pgfqpoint{1.112cm}{1.274cm}}
\pgfpathcurveto{\pgfqpoint{1.138cm}{1.249cm}}{\pgfqpoint{1.172cm}{1.234cm}}{\pgfqpoint{1.209cm}{1.234cm}}
\pgfpathcurveto{\pgfqpoint{1.245cm}{1.234cm}}{\pgfqpoint{1.28cm}{1.249cm}}{\pgfqpoint{1.305cm}{1.274cm}}
\pgfpathcurveto{\pgfqpoint{1.331cm}{1.3cm}}{\pgfqpoint{1.345cm}{1.335cm}}{\pgfqpoint{1.345cm}{1.371cm}}
\pgfusepath{fill}
\begin{pgfscope}
\pgfsetdash{}{0cm}
\pgfsetlinewidth{0.818mm}
\pgfsetroundcap
\pgfsetmiterlimit{4.0}
\pgfpathmoveto{\pgfqpoint{0.682cm}{0.671cm}}
\pgfpathlineto{\pgfqpoint{0.682cm}{0.042cm}}
\pgfusepath{stroke}
\end{pgfscope}
\end{pgfscope}
\end{pgfscope}
\end{pgfscope}
\end{tikzpicture}}}}-\rmm{3\UU_>\llbracket X^2 \rrbracket\succ(\phi+\psi)}+\rmb{3\UU_\leq\llbracket X^2 \rrbracket\succ(\phi+\psi)}-3[\Q,(-X^{\!\resizebox{0.6em}{!}{
\begin{tikzpicture}
\pgfpathmoveto{\pgfqpoint{0cm}{-0.035cm}}
\pgfpathlineto{\pgfqpoint{1.376cm}{-0.035cm}}
\pgfpathlineto{\pgfqpoint{1.376cm}{1.552cm}}
\pgfpathlineto{\pgfqpoint{0cm}{1.552cm}}
\pgfpathclose
\pgfusepath{clip}
\begin{pgfscope}
\begin{pgfscope}
\pgfpathmoveto{\pgfqpoint{0cm}{-0.035cm}}
\pgfpathlineto{\pgfqpoint{1.376cm}{-0.035cm}}
\pgfpathlineto{\pgfqpoint{1.376cm}{1.552cm}}
\pgfpathlineto{\pgfqpoint{0cm}{1.552cm}}
\pgfpathclose
\pgfusepath{clip}
\begin{pgfscope}
\begin{pgfscope}
\pgfsetdash{}{0cm}
\pgfsetlinewidth{0.818mm}
\pgfsetroundcap
\pgfsetroundjoin
\pgfsetmiterlimit{7.0}
\definecolor{eps2pgf_color}{gray}{0}\pgfsetstrokecolor{eps2pgf_color}\pgfsetfillcolor{eps2pgf_color}
\pgfpathmoveto{\pgfqpoint{0.117cm}{1.421cm}}
\pgfpathlineto{\pgfqpoint{0.682cm}{0.671cm}}
\pgfpathlineto{\pgfqpoint{1.246cm}{1.421cm}}
\pgfusepath{stroke}
\end{pgfscope}
\definecolor{eps2pgf_color}{gray}{0}\pgfsetstrokecolor{eps2pgf_color}\pgfsetfillcolor{eps2pgf_color}
\pgfpathmoveto{\pgfqpoint{0.273cm}{1.395cm}}
\pgfpathcurveto{\pgfqpoint{0.273cm}{1.432cm}}{\pgfqpoint{0.259cm}{1.467cm}}{\pgfqpoint{0.233cm}{1.492cm}}
\pgfpathcurveto{\pgfqpoint{0.207cm}{1.518cm}}{\pgfqpoint{0.173cm}{1.532cm}}{\pgfqpoint{0.137cm}{1.532cm}}
\pgfpathcurveto{\pgfqpoint{0.1cm}{1.532cm}}{\pgfqpoint{0.066cm}{1.518cm}}{\pgfqpoint{0.04cm}{1.492cm}}
\pgfpathcurveto{\pgfqpoint{0.014cm}{1.467cm}}{\pgfqpoint{0cm}{1.432cm}}{\pgfqpoint{0cm}{1.395cm}}
\pgfpathcurveto{\pgfqpoint{0cm}{1.359cm}}{\pgfqpoint{0.014cm}{1.324cm}}{\pgfqpoint{0.04cm}{1.299cm}}
\pgfpathcurveto{\pgfqpoint{0.066cm}{1.273cm}}{\pgfqpoint{0.1cm}{1.258cm}}{\pgfqpoint{0.137cm}{1.258cm}}
\pgfpathcurveto{\pgfqpoint{0.173cm}{1.258cm}}{\pgfqpoint{0.207cm}{1.273cm}}{\pgfqpoint{0.233cm}{1.299cm}}
\pgfpathcurveto{\pgfqpoint{0.259cm}{1.324cm}}{\pgfqpoint{0.273cm}{1.359cm}}{\pgfqpoint{0.273cm}{1.395cm}}
\pgfusepath{fill}
\begin{pgfscope}
\pgfsetdash{}{0cm}
\pgfsetlinewidth{0.818mm}
\pgfsetmiterlimit{7.0}
\pgfpathmoveto{\pgfqpoint{0.682cm}{0.671cm}}
\pgfpathlineto{\pgfqpoint{0.679cm}{1.418cm}}
\pgfusepath{stroke}
\end{pgfscope}
\pgfpathmoveto{\pgfqpoint{0.815cm}{1.399cm}}
\pgfpathcurveto{\pgfqpoint{0.815cm}{1.435cm}}{\pgfqpoint{0.801cm}{1.47cm}}{\pgfqpoint{0.775cm}{1.496cm}}
\pgfpathcurveto{\pgfqpoint{0.75cm}{1.521cm}}{\pgfqpoint{0.715cm}{1.536cm}}{\pgfqpoint{0.679cm}{1.536cm}}
\pgfpathcurveto{\pgfqpoint{0.643cm}{1.536cm}}{\pgfqpoint{0.608cm}{1.521cm}}{\pgfqpoint{0.582cm}{1.496cm}}
\pgfpathcurveto{\pgfqpoint{0.557cm}{1.47cm}}{\pgfqpoint{0.542cm}{1.435cm}}{\pgfqpoint{0.542cm}{1.399cm}}
\pgfpathcurveto{\pgfqpoint{0.542cm}{1.363cm}}{\pgfqpoint{0.557cm}{1.328cm}}{\pgfqpoint{0.582cm}{1.302cm}}
\pgfpathcurveto{\pgfqpoint{0.608cm}{1.276cm}}{\pgfqpoint{0.643cm}{1.262cm}}{\pgfqpoint{0.679cm}{1.262cm}}
\pgfpathcurveto{\pgfqpoint{0.715cm}{1.262cm}}{\pgfqpoint{0.75cm}{1.276cm}}{\pgfqpoint{0.775cm}{1.302cm}}
\pgfpathcurveto{\pgfqpoint{0.801cm}{1.328cm}}{\pgfqpoint{0.815cm}{1.363cm}}{\pgfqpoint{0.815cm}{1.399cm}}
\pgfusepath{fill}
\pgfpathmoveto{\pgfqpoint{1.345cm}{1.371cm}}
\pgfpathcurveto{\pgfqpoint{1.345cm}{1.408cm}}{\pgfqpoint{1.331cm}{1.442cm}}{\pgfqpoint{1.305cm}{1.468cm}}
\pgfpathcurveto{\pgfqpoint{1.28cm}{1.494cm}}{\pgfqpoint{1.245cm}{1.508cm}}{\pgfqpoint{1.209cm}{1.508cm}}
\pgfpathcurveto{\pgfqpoint{1.172cm}{1.508cm}}{\pgfqpoint{1.138cm}{1.494cm}}{\pgfqpoint{1.112cm}{1.468cm}}
\pgfpathcurveto{\pgfqpoint{1.087cm}{1.442cm}}{\pgfqpoint{1.072cm}{1.408cm}}{\pgfqpoint{1.072cm}{1.371cm}}
\pgfpathcurveto{\pgfqpoint{1.072cm}{1.335cm}}{\pgfqpoint{1.087cm}{1.3cm}}{\pgfqpoint{1.112cm}{1.274cm}}
\pgfpathcurveto{\pgfqpoint{1.138cm}{1.249cm}}{\pgfqpoint{1.172cm}{1.234cm}}{\pgfqpoint{1.209cm}{1.234cm}}
\pgfpathcurveto{\pgfqpoint{1.245cm}{1.234cm}}{\pgfqpoint{1.28cm}{1.249cm}}{\pgfqpoint{1.305cm}{1.274cm}}
\pgfpathcurveto{\pgfqpoint{1.331cm}{1.3cm}}{\pgfqpoint{1.345cm}{1.335cm}}{\pgfqpoint{1.345cm}{1.371cm}}
\pgfusepath{fill}
\begin{pgfscope}
\pgfsetdash{}{0cm}
\pgfsetlinewidth{0.818mm}
\pgfsetroundcap
\pgfsetmiterlimit{4.0}
\pgfpathmoveto{\pgfqpoint{0.682cm}{0.671cm}}
\pgfpathlineto{\pgfqpoint{0.682cm}{0.042cm}}
\pgfusepath{stroke}
\end{pgfscope}
\end{pgfscope}
\end{pgfscope}
\end{pgfscope}
\end{tikzpicture}}}+\phi+\psi)\prec]X^{\!\resizebox{0.6em}{!}{
\begin{tikzpicture}
\pgfpathmoveto{\pgfqpoint{0cm}{0cm}}
\pgfpathlineto{\pgfqpoint{1.376cm}{0cm}}
\pgfpathlineto{\pgfqpoint{1.376cm}{1.588cm}}
\pgfpathlineto{\pgfqpoint{0cm}{1.588cm}}
\pgfpathclose
\pgfusepath{clip}
\begin{pgfscope}
\begin{pgfscope}
\pgfpathmoveto{\pgfqpoint{0cm}{0cm}}
\pgfpathlineto{\pgfqpoint{1.376cm}{0cm}}
\pgfpathlineto{\pgfqpoint{1.376cm}{1.588cm}}
\pgfpathlineto{\pgfqpoint{0cm}{1.588cm}}
\pgfpathclose
\pgfusepath{clip}
\begin{pgfscope}
\begin{pgfscope}
\definecolor{eps2pgf_color}{gray}{0.976471}\pgfsetstrokecolor{eps2pgf_color}\pgfsetfillcolor{eps2pgf_color}
\pgfpathmoveto{\pgfqpoint{0cm}{0cm}}
\pgfpathlineto{\pgfqpoint{1.376cm}{0cm}}
\pgfpathlineto{\pgfqpoint{1.376cm}{1.588cm}}
\pgfpathlineto{\pgfqpoint{0cm}{1.588cm}}
\pgfpathclose
\pgfusepath{fill}
\end{pgfscope}
\begin{pgfscope}
\pgfsetdash{}{0cm}
\pgfsetlinewidth{0.818mm}
\pgfsetroundcap
\pgfsetroundjoin
\pgfsetmiterlimit{7.0}
\definecolor{eps2pgf_color}{gray}{0}\pgfsetstrokecolor{eps2pgf_color}\pgfsetfillcolor{eps2pgf_color}
\pgfpathmoveto{\pgfqpoint{0.117cm}{1.476cm}}
\pgfpathlineto{\pgfqpoint{0.682cm}{0.726cm}}
\pgfpathlineto{\pgfqpoint{1.246cm}{1.476cm}}
\pgfusepath{stroke}
\end{pgfscope}
\definecolor{eps2pgf_color}{gray}{0}\pgfsetstrokecolor{eps2pgf_color}\pgfsetfillcolor{eps2pgf_color}
\pgfpathmoveto{\pgfqpoint{0.273cm}{1.451cm}}
\pgfpathcurveto{\pgfqpoint{0.273cm}{1.487cm}}{\pgfqpoint{0.259cm}{1.522cm}}{\pgfqpoint{0.233cm}{1.547cm}}
\pgfpathcurveto{\pgfqpoint{0.207cm}{1.573cm}}{\pgfqpoint{0.173cm}{1.588cm}}{\pgfqpoint{0.137cm}{1.588cm}}
\pgfpathcurveto{\pgfqpoint{0.1cm}{1.588cm}}{\pgfqpoint{0.066cm}{1.573cm}}{\pgfqpoint{0.04cm}{1.547cm}}
\pgfpathcurveto{\pgfqpoint{0.014cm}{1.522cm}}{\pgfqpoint{0cm}{1.487cm}}{\pgfqpoint{0cm}{1.451cm}}
\pgfpathcurveto{\pgfqpoint{0cm}{1.414cm}}{\pgfqpoint{0.014cm}{1.379cm}}{\pgfqpoint{0.04cm}{1.354cm}}
\pgfpathcurveto{\pgfqpoint{0.066cm}{1.328cm}}{\pgfqpoint{0.1cm}{1.314cm}}{\pgfqpoint{0.137cm}{1.314cm}}
\pgfpathcurveto{\pgfqpoint{0.173cm}{1.314cm}}{\pgfqpoint{0.207cm}{1.328cm}}{\pgfqpoint{0.233cm}{1.354cm}}
\pgfpathcurveto{\pgfqpoint{0.259cm}{1.379cm}}{\pgfqpoint{0.273cm}{1.414cm}}{\pgfqpoint{0.273cm}{1.451cm}}
\pgfusepath{fill}
\pgfpathmoveto{\pgfqpoint{1.345cm}{1.426cm}}
\pgfpathcurveto{\pgfqpoint{1.345cm}{1.463cm}}{\pgfqpoint{1.331cm}{1.497cm}}{\pgfqpoint{1.305cm}{1.523cm}}
\pgfpathcurveto{\pgfqpoint{1.28cm}{1.549cm}}{\pgfqpoint{1.245cm}{1.563cm}}{\pgfqpoint{1.209cm}{1.563cm}}
\pgfpathcurveto{\pgfqpoint{1.172cm}{1.563cm}}{\pgfqpoint{1.138cm}{1.549cm}}{\pgfqpoint{1.112cm}{1.523cm}}
\pgfpathcurveto{\pgfqpoint{1.087cm}{1.497cm}}{\pgfqpoint{1.072cm}{1.463cm}}{\pgfqpoint{1.072cm}{1.426cm}}
\pgfpathcurveto{\pgfqpoint{1.072cm}{1.39cm}}{\pgfqpoint{1.087cm}{1.355cm}}{\pgfqpoint{1.112cm}{1.329cm}}
\pgfpathcurveto{\pgfqpoint{1.138cm}{1.304cm}}{\pgfqpoint{1.172cm}{1.289cm}}{\pgfqpoint{1.209cm}{1.289cm}}
\pgfpathcurveto{\pgfqpoint{1.245cm}{1.289cm}}{\pgfqpoint{1.28cm}{1.304cm}}{\pgfqpoint{1.305cm}{1.329cm}}
\pgfpathcurveto{\pgfqpoint{1.331cm}{1.355cm}}{\pgfqpoint{1.345cm}{1.39cm}}{\pgfqpoint{1.345cm}{1.426cm}}
\pgfusepath{fill}
\begin{pgfscope}
\pgfsetdash{}{0cm}
\pgfsetlinewidth{0.818mm}
\pgfsetroundcap
\pgfsetmiterlimit{4.0}
\pgfpathmoveto{\pgfqpoint{0.682cm}{0.726cm}}
\pgfpathlineto{\pgfqpoint{0.682cm}{0.097cm}}
\pgfusepath{stroke}
\end{pgfscope}
\end{pgfscope}
\end{pgfscope}
\end{pgfscope}
\end{tikzpicture}}}.
$$
Next, we observe that all the remaining terms from $\Phi$ can be estimated in $\CC^{-1+\alpha}(\rho^{2+\alpha})$. Indeed, all the terms that do not contain $\phi+\psi$ are bounded in this space and for the terms containing $\phi+\psi$, we observe that
$$
\|\UU_> \llbracket X^2 \rrbracket\|_{\CC^{-\frac{3}{2}}(\rho^{\frac{1}{2}})}\lesssim 2^{-(\frac{1}{2}-\kappa)K} \|\llbracket X^2 \rrbracket\|_{\CC^{-1-\kappa}(\rho^\sigma)},
$$
$$
\|\UU_\leq\llbracket X^2 \rrbracket\|_{\CC^{-1+\alpha}(\rho^{1+\alpha})}\lesssim 2^{(\alpha+\kappa)K} \|\llbracket X^2 \rrbracket\|_{\CC^{-1-\kappa}(\rho^{\sigma})},
$$
$$
\|\UU_> X^{\!\resizebox{!}{.8em}{
\begin{tikzpicture}
\pgfpathmoveto{\pgfqpoint{0cm}{-0.035cm}}
\pgfpathlineto{\pgfqpoint{1.976cm}{-0.035cm}}
\pgfpathlineto{\pgfqpoint{1.976cm}{1.94cm}}
\pgfpathlineto{\pgfqpoint{0cm}{1.94cm}}
\pgfpathclose
\pgfusepath{clip}
\begin{pgfscope}
\begin{pgfscope}
\pgfpathmoveto{\pgfqpoint{0cm}{-0.035cm}}
\pgfpathlineto{\pgfqpoint{1.976cm}{-0.035cm}}
\pgfpathlineto{\pgfqpoint{1.976cm}{1.94cm}}
\pgfpathlineto{\pgfqpoint{0cm}{1.94cm}}
\pgfpathclose
\pgfusepath{clip}
\begin{pgfscope}
\begin{pgfscope}
\pgfsetdash{}{0cm}
\pgfsetlinewidth{0.818mm}
\pgfsetroundcap
\pgfsetroundjoin
\pgfsetmiterlimit{7.0}
\definecolor{eps2pgf_color}{gray}{0}\pgfsetstrokecolor{eps2pgf_color}\pgfsetfillcolor{eps2pgf_color}
\pgfpathmoveto{\pgfqpoint{0.117cm}{1.815cm}}
\pgfpathlineto{\pgfqpoint{0.682cm}{1.065cm}}
\pgfpathlineto{\pgfqpoint{1.246cm}{1.815cm}}
\pgfusepath{stroke}
\end{pgfscope}
\definecolor{eps2pgf_color}{gray}{0}\pgfsetstrokecolor{eps2pgf_color}\pgfsetfillcolor{eps2pgf_color}
\pgfpathmoveto{\pgfqpoint{0.273cm}{1.789cm}}
\pgfpathcurveto{\pgfqpoint{0.273cm}{1.825cm}}{\pgfqpoint{0.259cm}{1.86cm}}{\pgfqpoint{0.233cm}{1.886cm}}
\pgfpathcurveto{\pgfqpoint{0.207cm}{1.912cm}}{\pgfqpoint{0.173cm}{1.926cm}}{\pgfqpoint{0.137cm}{1.926cm}}
\pgfpathcurveto{\pgfqpoint{0.1cm}{1.926cm}}{\pgfqpoint{0.066cm}{1.912cm}}{\pgfqpoint{0.04cm}{1.886cm}}
\pgfpathcurveto{\pgfqpoint{0.014cm}{1.86cm}}{\pgfqpoint{0cm}{1.825cm}}{\pgfqpoint{0cm}{1.789cm}}
\pgfpathcurveto{\pgfqpoint{0cm}{1.753cm}}{\pgfqpoint{0.014cm}{1.718cm}}{\pgfqpoint{0.04cm}{1.692cm}}
\pgfpathcurveto{\pgfqpoint{0.066cm}{1.667cm}}{\pgfqpoint{0.1cm}{1.652cm}}{\pgfqpoint{0.137cm}{1.652cm}}
\pgfpathcurveto{\pgfqpoint{0.173cm}{1.652cm}}{\pgfqpoint{0.207cm}{1.667cm}}{\pgfqpoint{0.233cm}{1.692cm}}
\pgfpathcurveto{\pgfqpoint{0.259cm}{1.718cm}}{\pgfqpoint{0.273cm}{1.753cm}}{\pgfqpoint{0.273cm}{1.789cm}}
\pgfusepath{fill}
\pgfpathmoveto{\pgfqpoint{1.345cm}{1.765cm}}
\pgfpathcurveto{\pgfqpoint{1.345cm}{1.801cm}}{\pgfqpoint{1.331cm}{1.836cm}}{\pgfqpoint{1.305cm}{1.862cm}}
\pgfpathcurveto{\pgfqpoint{1.28cm}{1.887cm}}{\pgfqpoint{1.245cm}{1.902cm}}{\pgfqpoint{1.209cm}{1.902cm}}
\pgfpathcurveto{\pgfqpoint{1.172cm}{1.902cm}}{\pgfqpoint{1.138cm}{1.887cm}}{\pgfqpoint{1.112cm}{1.862cm}}
\pgfpathcurveto{\pgfqpoint{1.087cm}{1.836cm}}{\pgfqpoint{1.072cm}{1.801cm}}{\pgfqpoint{1.072cm}{1.765cm}}
\pgfpathcurveto{\pgfqpoint{1.072cm}{1.728cm}}{\pgfqpoint{1.087cm}{1.694cm}}{\pgfqpoint{1.112cm}{1.668cm}}
\pgfpathcurveto{\pgfqpoint{1.138cm}{1.642cm}}{\pgfqpoint{1.172cm}{1.628cm}}{\pgfqpoint{1.209cm}{1.628cm}}
\pgfpathcurveto{\pgfqpoint{1.245cm}{1.628cm}}{\pgfqpoint{1.28cm}{1.642cm}}{\pgfqpoint{1.305cm}{1.668cm}}
\pgfpathcurveto{\pgfqpoint{1.331cm}{1.694cm}}{\pgfqpoint{1.345cm}{1.728cm}}{\pgfqpoint{1.345cm}{1.765cm}}
\pgfusepath{fill}
\begin{pgfscope}
\pgfsetdash{}{0cm}
\pgfsetlinewidth{0.818mm}
\pgfsetroundcap
\pgfsetroundjoin
\pgfsetmiterlimit{7.0}
\pgfpathmoveto{\pgfqpoint{0.682cm}{1.065cm}}
\pgfpathlineto{\pgfqpoint{1.246cm}{0.315cm}}
\pgfpathlineto{\pgfqpoint{1.811cm}{1.065cm}}
\pgfusepath{stroke}
\end{pgfscope}
\pgfpathmoveto{\pgfqpoint{1.948cm}{1.065cm}}
\pgfpathcurveto{\pgfqpoint{1.948cm}{1.101cm}}{\pgfqpoint{1.933cm}{1.136cm}}{\pgfqpoint{1.907cm}{1.162cm}}
\pgfpathcurveto{\pgfqpoint{1.882cm}{1.187cm}}{\pgfqpoint{1.847cm}{1.202cm}}{\pgfqpoint{1.811cm}{1.202cm}}
\pgfpathcurveto{\pgfqpoint{1.775cm}{1.202cm}}{\pgfqpoint{1.74cm}{1.187cm}}{\pgfqpoint{1.714cm}{1.162cm}}
\pgfpathcurveto{\pgfqpoint{1.689cm}{1.136cm}}{\pgfqpoint{1.674cm}{1.101cm}}{\pgfqpoint{1.674cm}{1.065cm}}
\pgfpathcurveto{\pgfqpoint{1.674cm}{1.029cm}}{\pgfqpoint{1.689cm}{0.994cm}}{\pgfqpoint{1.714cm}{0.968cm}}
\pgfpathcurveto{\pgfqpoint{1.74cm}{0.942cm}}{\pgfqpoint{1.775cm}{0.928cm}}{\pgfqpoint{1.811cm}{0.928cm}}
\pgfpathcurveto{\pgfqpoint{1.847cm}{0.928cm}}{\pgfqpoint{1.882cm}{0.942cm}}{\pgfqpoint{1.907cm}{0.968cm}}
\pgfpathcurveto{\pgfqpoint{1.933cm}{0.994cm}}{\pgfqpoint{1.948cm}{1.029cm}}{\pgfqpoint{1.948cm}{1.065cm}}
\pgfusepath{fill}
\begin{pgfscope}
\pgfsetdash{}{0cm}
\pgfsetlinewidth{0.818mm}
\pgfsetmiterlimit{7.0}
\pgfpathmoveto{\pgfqpoint{1.246cm}{0.315cm}}
\pgfpathlineto{\pgfqpoint{1.244cm}{1.061cm}}
\pgfusepath{stroke}
\end{pgfscope}
\pgfpathmoveto{\pgfqpoint{1.38cm}{1.065cm}}
\pgfpathcurveto{\pgfqpoint{1.38cm}{1.101cm}}{\pgfqpoint{1.366cm}{1.136cm}}{\pgfqpoint{1.34cm}{1.162cm}}
\pgfpathcurveto{\pgfqpoint{1.315cm}{1.187cm}}{\pgfqpoint{1.28cm}{1.202cm}}{\pgfqpoint{1.244cm}{1.202cm}}
\pgfpathcurveto{\pgfqpoint{1.207cm}{1.202cm}}{\pgfqpoint{1.173cm}{1.187cm}}{\pgfqpoint{1.147cm}{1.162cm}}
\pgfpathcurveto{\pgfqpoint{1.121cm}{1.136cm}}{\pgfqpoint{1.107cm}{1.101cm}}{\pgfqpoint{1.107cm}{1.065cm}}
\pgfpathcurveto{\pgfqpoint{1.107cm}{1.029cm}}{\pgfqpoint{1.121cm}{0.994cm}}{\pgfqpoint{1.147cm}{0.968cm}}
\pgfpathcurveto{\pgfqpoint{1.173cm}{0.942cm}}{\pgfqpoint{1.207cm}{0.928cm}}{\pgfqpoint{1.244cm}{0.928cm}}
\pgfpathcurveto{\pgfqpoint{1.28cm}{0.928cm}}{\pgfqpoint{1.315cm}{0.942cm}}{\pgfqpoint{1.34cm}{0.968cm}}
\pgfpathcurveto{\pgfqpoint{1.366cm}{0.994cm}}{\pgfqpoint{1.38cm}{1.029cm}}{\pgfqpoint{1.38cm}{1.065cm}}
\pgfusepath{fill}
\begin{pgfscope}
\pgfsetdash{}{0cm}
\pgfsetlinewidth{0.818mm}
\pgfsetmiterlimit{4.0}
\pgfpathmoveto{\pgfqpoint{1.383cm}{0.178cm}}
\pgfpathcurveto{\pgfqpoint{1.383cm}{0.214cm}}{\pgfqpoint{1.369cm}{0.249cm}}{\pgfqpoint{1.343cm}{0.275cm}}
\pgfpathcurveto{\pgfqpoint{1.317cm}{0.3cm}}{\pgfqpoint{1.283cm}{0.315cm}}{\pgfqpoint{1.246cm}{0.315cm}}
\pgfpathcurveto{\pgfqpoint{1.21cm}{0.315cm}}{\pgfqpoint{1.175cm}{0.3cm}}{\pgfqpoint{1.15cm}{0.275cm}}
\pgfpathcurveto{\pgfqpoint{1.124cm}{0.249cm}}{\pgfqpoint{1.11cm}{0.214cm}}{\pgfqpoint{1.11cm}{0.178cm}}
\pgfpathcurveto{\pgfqpoint{1.11cm}{0.141cm}}{\pgfqpoint{1.124cm}{0.107cm}}{\pgfqpoint{1.15cm}{0.081cm}}
\pgfpathcurveto{\pgfqpoint{1.175cm}{0.055cm}}{\pgfqpoint{1.21cm}{0.041cm}}{\pgfqpoint{1.246cm}{0.041cm}}
\pgfpathcurveto{\pgfqpoint{1.283cm}{0.041cm}}{\pgfqpoint{1.317cm}{0.055cm}}{\pgfqpoint{1.343cm}{0.081cm}}
\pgfpathcurveto{\pgfqpoint{1.369cm}{0.107cm}}{\pgfqpoint{1.383cm}{0.141cm}}{\pgfqpoint{1.383cm}{0.178cm}}
\pgfusepath{stroke}
\end{pgfscope}
\end{pgfscope}
\end{pgfscope}
\end{pgfscope}
\end{tikzpicture}}}\|_{\CC^{-1+\alpha}}\lesssim 2^{-(1-\alpha-\kappa)K/2} \|X^{\!\resizebox{!}{.8em}{
\begin{tikzpicture}
\pgfpathmoveto{\pgfqpoint{0cm}{-0.035cm}}
\pgfpathlineto{\pgfqpoint{1.976cm}{-0.035cm}}
\pgfpathlineto{\pgfqpoint{1.976cm}{1.94cm}}
\pgfpathlineto{\pgfqpoint{0cm}{1.94cm}}
\pgfpathclose
\pgfusepath{clip}
\begin{pgfscope}
\begin{pgfscope}
\pgfpathmoveto{\pgfqpoint{0cm}{-0.035cm}}
\pgfpathlineto{\pgfqpoint{1.976cm}{-0.035cm}}
\pgfpathlineto{\pgfqpoint{1.976cm}{1.94cm}}
\pgfpathlineto{\pgfqpoint{0cm}{1.94cm}}
\pgfpathclose
\pgfusepath{clip}
\begin{pgfscope}
\begin{pgfscope}
\pgfsetdash{}{0cm}
\pgfsetlinewidth{0.818mm}
\pgfsetroundcap
\pgfsetroundjoin
\pgfsetmiterlimit{7.0}
\definecolor{eps2pgf_color}{gray}{0}\pgfsetstrokecolor{eps2pgf_color}\pgfsetfillcolor{eps2pgf_color}
\pgfpathmoveto{\pgfqpoint{0.117cm}{1.815cm}}
\pgfpathlineto{\pgfqpoint{0.682cm}{1.065cm}}
\pgfpathlineto{\pgfqpoint{1.246cm}{1.815cm}}
\pgfusepath{stroke}
\end{pgfscope}
\definecolor{eps2pgf_color}{gray}{0}\pgfsetstrokecolor{eps2pgf_color}\pgfsetfillcolor{eps2pgf_color}
\pgfpathmoveto{\pgfqpoint{0.273cm}{1.789cm}}
\pgfpathcurveto{\pgfqpoint{0.273cm}{1.825cm}}{\pgfqpoint{0.259cm}{1.86cm}}{\pgfqpoint{0.233cm}{1.886cm}}
\pgfpathcurveto{\pgfqpoint{0.207cm}{1.912cm}}{\pgfqpoint{0.173cm}{1.926cm}}{\pgfqpoint{0.137cm}{1.926cm}}
\pgfpathcurveto{\pgfqpoint{0.1cm}{1.926cm}}{\pgfqpoint{0.066cm}{1.912cm}}{\pgfqpoint{0.04cm}{1.886cm}}
\pgfpathcurveto{\pgfqpoint{0.014cm}{1.86cm}}{\pgfqpoint{0cm}{1.825cm}}{\pgfqpoint{0cm}{1.789cm}}
\pgfpathcurveto{\pgfqpoint{0cm}{1.753cm}}{\pgfqpoint{0.014cm}{1.718cm}}{\pgfqpoint{0.04cm}{1.692cm}}
\pgfpathcurveto{\pgfqpoint{0.066cm}{1.667cm}}{\pgfqpoint{0.1cm}{1.652cm}}{\pgfqpoint{0.137cm}{1.652cm}}
\pgfpathcurveto{\pgfqpoint{0.173cm}{1.652cm}}{\pgfqpoint{0.207cm}{1.667cm}}{\pgfqpoint{0.233cm}{1.692cm}}
\pgfpathcurveto{\pgfqpoint{0.259cm}{1.718cm}}{\pgfqpoint{0.273cm}{1.753cm}}{\pgfqpoint{0.273cm}{1.789cm}}
\pgfusepath{fill}
\pgfpathmoveto{\pgfqpoint{1.345cm}{1.765cm}}
\pgfpathcurveto{\pgfqpoint{1.345cm}{1.801cm}}{\pgfqpoint{1.331cm}{1.836cm}}{\pgfqpoint{1.305cm}{1.862cm}}
\pgfpathcurveto{\pgfqpoint{1.28cm}{1.887cm}}{\pgfqpoint{1.245cm}{1.902cm}}{\pgfqpoint{1.209cm}{1.902cm}}
\pgfpathcurveto{\pgfqpoint{1.172cm}{1.902cm}}{\pgfqpoint{1.138cm}{1.887cm}}{\pgfqpoint{1.112cm}{1.862cm}}
\pgfpathcurveto{\pgfqpoint{1.087cm}{1.836cm}}{\pgfqpoint{1.072cm}{1.801cm}}{\pgfqpoint{1.072cm}{1.765cm}}
\pgfpathcurveto{\pgfqpoint{1.072cm}{1.728cm}}{\pgfqpoint{1.087cm}{1.694cm}}{\pgfqpoint{1.112cm}{1.668cm}}
\pgfpathcurveto{\pgfqpoint{1.138cm}{1.642cm}}{\pgfqpoint{1.172cm}{1.628cm}}{\pgfqpoint{1.209cm}{1.628cm}}
\pgfpathcurveto{\pgfqpoint{1.245cm}{1.628cm}}{\pgfqpoint{1.28cm}{1.642cm}}{\pgfqpoint{1.305cm}{1.668cm}}
\pgfpathcurveto{\pgfqpoint{1.331cm}{1.694cm}}{\pgfqpoint{1.345cm}{1.728cm}}{\pgfqpoint{1.345cm}{1.765cm}}
\pgfusepath{fill}
\begin{pgfscope}
\pgfsetdash{}{0cm}
\pgfsetlinewidth{0.818mm}
\pgfsetroundcap
\pgfsetroundjoin
\pgfsetmiterlimit{7.0}
\pgfpathmoveto{\pgfqpoint{0.682cm}{1.065cm}}
\pgfpathlineto{\pgfqpoint{1.246cm}{0.315cm}}
\pgfpathlineto{\pgfqpoint{1.811cm}{1.065cm}}
\pgfusepath{stroke}
\end{pgfscope}
\pgfpathmoveto{\pgfqpoint{1.948cm}{1.065cm}}
\pgfpathcurveto{\pgfqpoint{1.948cm}{1.101cm}}{\pgfqpoint{1.933cm}{1.136cm}}{\pgfqpoint{1.907cm}{1.162cm}}
\pgfpathcurveto{\pgfqpoint{1.882cm}{1.187cm}}{\pgfqpoint{1.847cm}{1.202cm}}{\pgfqpoint{1.811cm}{1.202cm}}
\pgfpathcurveto{\pgfqpoint{1.775cm}{1.202cm}}{\pgfqpoint{1.74cm}{1.187cm}}{\pgfqpoint{1.714cm}{1.162cm}}
\pgfpathcurveto{\pgfqpoint{1.689cm}{1.136cm}}{\pgfqpoint{1.674cm}{1.101cm}}{\pgfqpoint{1.674cm}{1.065cm}}
\pgfpathcurveto{\pgfqpoint{1.674cm}{1.029cm}}{\pgfqpoint{1.689cm}{0.994cm}}{\pgfqpoint{1.714cm}{0.968cm}}
\pgfpathcurveto{\pgfqpoint{1.74cm}{0.942cm}}{\pgfqpoint{1.775cm}{0.928cm}}{\pgfqpoint{1.811cm}{0.928cm}}
\pgfpathcurveto{\pgfqpoint{1.847cm}{0.928cm}}{\pgfqpoint{1.882cm}{0.942cm}}{\pgfqpoint{1.907cm}{0.968cm}}
\pgfpathcurveto{\pgfqpoint{1.933cm}{0.994cm}}{\pgfqpoint{1.948cm}{1.029cm}}{\pgfqpoint{1.948cm}{1.065cm}}
\pgfusepath{fill}
\begin{pgfscope}
\pgfsetdash{}{0cm}
\pgfsetlinewidth{0.818mm}
\pgfsetmiterlimit{7.0}
\pgfpathmoveto{\pgfqpoint{1.246cm}{0.315cm}}
\pgfpathlineto{\pgfqpoint{1.244cm}{1.061cm}}
\pgfusepath{stroke}
\end{pgfscope}
\pgfpathmoveto{\pgfqpoint{1.38cm}{1.065cm}}
\pgfpathcurveto{\pgfqpoint{1.38cm}{1.101cm}}{\pgfqpoint{1.366cm}{1.136cm}}{\pgfqpoint{1.34cm}{1.162cm}}
\pgfpathcurveto{\pgfqpoint{1.315cm}{1.187cm}}{\pgfqpoint{1.28cm}{1.202cm}}{\pgfqpoint{1.244cm}{1.202cm}}
\pgfpathcurveto{\pgfqpoint{1.207cm}{1.202cm}}{\pgfqpoint{1.173cm}{1.187cm}}{\pgfqpoint{1.147cm}{1.162cm}}
\pgfpathcurveto{\pgfqpoint{1.121cm}{1.136cm}}{\pgfqpoint{1.107cm}{1.101cm}}{\pgfqpoint{1.107cm}{1.065cm}}
\pgfpathcurveto{\pgfqpoint{1.107cm}{1.029cm}}{\pgfqpoint{1.121cm}{0.994cm}}{\pgfqpoint{1.147cm}{0.968cm}}
\pgfpathcurveto{\pgfqpoint{1.173cm}{0.942cm}}{\pgfqpoint{1.207cm}{0.928cm}}{\pgfqpoint{1.244cm}{0.928cm}}
\pgfpathcurveto{\pgfqpoint{1.28cm}{0.928cm}}{\pgfqpoint{1.315cm}{0.942cm}}{\pgfqpoint{1.34cm}{0.968cm}}
\pgfpathcurveto{\pgfqpoint{1.366cm}{0.994cm}}{\pgfqpoint{1.38cm}{1.029cm}}{\pgfqpoint{1.38cm}{1.065cm}}
\pgfusepath{fill}
\begin{pgfscope}
\pgfsetdash{}{0cm}
\pgfsetlinewidth{0.818mm}
\pgfsetmiterlimit{4.0}
\pgfpathmoveto{\pgfqpoint{1.383cm}{0.178cm}}
\pgfpathcurveto{\pgfqpoint{1.383cm}{0.214cm}}{\pgfqpoint{1.369cm}{0.249cm}}{\pgfqpoint{1.343cm}{0.275cm}}
\pgfpathcurveto{\pgfqpoint{1.317cm}{0.3cm}}{\pgfqpoint{1.283cm}{0.315cm}}{\pgfqpoint{1.246cm}{0.315cm}}
\pgfpathcurveto{\pgfqpoint{1.21cm}{0.315cm}}{\pgfqpoint{1.175cm}{0.3cm}}{\pgfqpoint{1.15cm}{0.275cm}}
\pgfpathcurveto{\pgfqpoint{1.124cm}{0.249cm}}{\pgfqpoint{1.11cm}{0.214cm}}{\pgfqpoint{1.11cm}{0.178cm}}
\pgfpathcurveto{\pgfqpoint{1.11cm}{0.141cm}}{\pgfqpoint{1.124cm}{0.107cm}}{\pgfqpoint{1.15cm}{0.081cm}}
\pgfpathcurveto{\pgfqpoint{1.175cm}{0.055cm}}{\pgfqpoint{1.21cm}{0.041cm}}{\pgfqpoint{1.246cm}{0.041cm}}
\pgfpathcurveto{\pgfqpoint{1.283cm}{0.041cm}}{\pgfqpoint{1.317cm}{0.055cm}}{\pgfqpoint{1.343cm}{0.081cm}}
\pgfpathcurveto{\pgfqpoint{1.369cm}{0.107cm}}{\pgfqpoint{1.383cm}{0.141cm}}{\pgfqpoint{1.383cm}{0.178cm}}
\pgfusepath{stroke}
\end{pgfscope}
\end{pgfscope}
\end{pgfscope}
\end{pgfscope}
\end{tikzpicture}}} \|_{\CC^{-\kappa}(\rho^\sigma)},
$$
$$
\|\UU_> X\|_{\CC^{-1+\alpha}(\rho^{-\alpha})}\lesssim 2^{-(\frac{1}{2}-\alpha-\kappa)\frac{2}{3}K} \|X\|_{\CC^{-\frac{1}{2}-\kappa}(\rho^\sigma)},
$$
$$
\|\UU_> X^{\!\resizebox{!}{.8em}{
\begin{tikzpicture}
\pgfpathmoveto{\pgfqpoint{0cm}{-0.035cm}}
\pgfpathlineto{\pgfqpoint{1.976cm}{-0.035cm}}
\pgfpathlineto{\pgfqpoint{1.976cm}{1.94cm}}
\pgfpathlineto{\pgfqpoint{0cm}{1.94cm}}
\pgfpathclose
\pgfusepath{clip}
\begin{pgfscope}
\begin{pgfscope}
\pgfpathmoveto{\pgfqpoint{0cm}{-0.035cm}}
\pgfpathlineto{\pgfqpoint{1.976cm}{-0.035cm}}
\pgfpathlineto{\pgfqpoint{1.976cm}{1.94cm}}
\pgfpathlineto{\pgfqpoint{0cm}{1.94cm}}
\pgfpathclose
\pgfusepath{clip}
\begin{pgfscope}
\begin{pgfscope}
\pgfsetdash{}{0cm}
\pgfsetlinewidth{0.818mm}
\pgfsetroundcap
\pgfsetroundjoin
\pgfsetmiterlimit{7.0}
\definecolor{eps2pgf_color}{gray}{0}\pgfsetstrokecolor{eps2pgf_color}\pgfsetfillcolor{eps2pgf_color}
\pgfpathmoveto{\pgfqpoint{0.117cm}{1.815cm}}
\pgfpathlineto{\pgfqpoint{0.682cm}{1.065cm}}
\pgfpathlineto{\pgfqpoint{1.246cm}{1.815cm}}
\pgfusepath{stroke}
\end{pgfscope}
\definecolor{eps2pgf_color}{gray}{0}\pgfsetstrokecolor{eps2pgf_color}\pgfsetfillcolor{eps2pgf_color}
\pgfpathmoveto{\pgfqpoint{0.273cm}{1.789cm}}
\pgfpathcurveto{\pgfqpoint{0.273cm}{1.825cm}}{\pgfqpoint{0.259cm}{1.86cm}}{\pgfqpoint{0.233cm}{1.886cm}}
\pgfpathcurveto{\pgfqpoint{0.207cm}{1.912cm}}{\pgfqpoint{0.173cm}{1.926cm}}{\pgfqpoint{0.137cm}{1.926cm}}
\pgfpathcurveto{\pgfqpoint{0.1cm}{1.926cm}}{\pgfqpoint{0.066cm}{1.912cm}}{\pgfqpoint{0.04cm}{1.886cm}}
\pgfpathcurveto{\pgfqpoint{0.014cm}{1.86cm}}{\pgfqpoint{0cm}{1.825cm}}{\pgfqpoint{0cm}{1.789cm}}
\pgfpathcurveto{\pgfqpoint{0cm}{1.753cm}}{\pgfqpoint{0.014cm}{1.718cm}}{\pgfqpoint{0.04cm}{1.692cm}}
\pgfpathcurveto{\pgfqpoint{0.066cm}{1.667cm}}{\pgfqpoint{0.1cm}{1.652cm}}{\pgfqpoint{0.137cm}{1.652cm}}
\pgfpathcurveto{\pgfqpoint{0.173cm}{1.652cm}}{\pgfqpoint{0.207cm}{1.667cm}}{\pgfqpoint{0.233cm}{1.692cm}}
\pgfpathcurveto{\pgfqpoint{0.259cm}{1.718cm}}{\pgfqpoint{0.273cm}{1.753cm}}{\pgfqpoint{0.273cm}{1.789cm}}
\pgfusepath{fill}
\begin{pgfscope}
\pgfsetdash{}{0cm}
\pgfsetlinewidth{0.818mm}
\pgfsetmiterlimit{7.0}
\pgfpathmoveto{\pgfqpoint{0.682cm}{1.065cm}}
\pgfpathlineto{\pgfqpoint{0.679cm}{1.812cm}}
\pgfusepath{stroke}
\end{pgfscope}
\pgfpathmoveto{\pgfqpoint{0.815cm}{1.793cm}}
\pgfpathcurveto{\pgfqpoint{0.815cm}{1.829cm}}{\pgfqpoint{0.801cm}{1.864cm}}{\pgfqpoint{0.775cm}{1.89cm}}
\pgfpathcurveto{\pgfqpoint{0.75cm}{1.915cm}}{\pgfqpoint{0.715cm}{1.93cm}}{\pgfqpoint{0.679cm}{1.93cm}}
\pgfpathcurveto{\pgfqpoint{0.643cm}{1.93cm}}{\pgfqpoint{0.608cm}{1.915cm}}{\pgfqpoint{0.582cm}{1.89cm}}
\pgfpathcurveto{\pgfqpoint{0.557cm}{1.864cm}}{\pgfqpoint{0.542cm}{1.829cm}}{\pgfqpoint{0.542cm}{1.793cm}}
\pgfpathcurveto{\pgfqpoint{0.542cm}{1.756cm}}{\pgfqpoint{0.557cm}{1.722cm}}{\pgfqpoint{0.582cm}{1.696cm}}
\pgfpathcurveto{\pgfqpoint{0.608cm}{1.67cm}}{\pgfqpoint{0.643cm}{1.656cm}}{\pgfqpoint{0.679cm}{1.656cm}}
\pgfpathcurveto{\pgfqpoint{0.715cm}{1.656cm}}{\pgfqpoint{0.75cm}{1.67cm}}{\pgfqpoint{0.775cm}{1.696cm}}
\pgfpathcurveto{\pgfqpoint{0.801cm}{1.722cm}}{\pgfqpoint{0.815cm}{1.756cm}}{\pgfqpoint{0.815cm}{1.793cm}}
\pgfusepath{fill}
\pgfpathmoveto{\pgfqpoint{1.345cm}{1.765cm}}
\pgfpathcurveto{\pgfqpoint{1.345cm}{1.801cm}}{\pgfqpoint{1.331cm}{1.836cm}}{\pgfqpoint{1.305cm}{1.862cm}}
\pgfpathcurveto{\pgfqpoint{1.28cm}{1.887cm}}{\pgfqpoint{1.245cm}{1.902cm}}{\pgfqpoint{1.209cm}{1.902cm}}
\pgfpathcurveto{\pgfqpoint{1.172cm}{1.902cm}}{\pgfqpoint{1.138cm}{1.887cm}}{\pgfqpoint{1.112cm}{1.862cm}}
\pgfpathcurveto{\pgfqpoint{1.087cm}{1.836cm}}{\pgfqpoint{1.072cm}{1.801cm}}{\pgfqpoint{1.072cm}{1.765cm}}
\pgfpathcurveto{\pgfqpoint{1.072cm}{1.728cm}}{\pgfqpoint{1.087cm}{1.694cm}}{\pgfqpoint{1.112cm}{1.668cm}}
\pgfpathcurveto{\pgfqpoint{1.138cm}{1.642cm}}{\pgfqpoint{1.172cm}{1.628cm}}{\pgfqpoint{1.209cm}{1.628cm}}
\pgfpathcurveto{\pgfqpoint{1.245cm}{1.628cm}}{\pgfqpoint{1.28cm}{1.642cm}}{\pgfqpoint{1.305cm}{1.668cm}}
\pgfpathcurveto{\pgfqpoint{1.331cm}{1.694cm}}{\pgfqpoint{1.345cm}{1.728cm}}{\pgfqpoint{1.345cm}{1.765cm}}
\pgfusepath{fill}
\begin{pgfscope}
\pgfsetdash{}{0cm}
\pgfsetlinewidth{0.818mm}
\pgfsetroundcap
\pgfsetroundjoin
\pgfsetmiterlimit{7.0}
\pgfpathmoveto{\pgfqpoint{0.682cm}{1.065cm}}
\pgfpathlineto{\pgfqpoint{1.246cm}{0.315cm}}
\pgfpathlineto{\pgfqpoint{1.811cm}{1.065cm}}
\pgfusepath{stroke}
\end{pgfscope}
\pgfpathmoveto{\pgfqpoint{1.948cm}{1.065cm}}
\pgfpathcurveto{\pgfqpoint{1.948cm}{1.101cm}}{\pgfqpoint{1.933cm}{1.136cm}}{\pgfqpoint{1.907cm}{1.162cm}}
\pgfpathcurveto{\pgfqpoint{1.882cm}{1.187cm}}{\pgfqpoint{1.847cm}{1.202cm}}{\pgfqpoint{1.811cm}{1.202cm}}
\pgfpathcurveto{\pgfqpoint{1.775cm}{1.202cm}}{\pgfqpoint{1.74cm}{1.187cm}}{\pgfqpoint{1.714cm}{1.162cm}}
\pgfpathcurveto{\pgfqpoint{1.689cm}{1.136cm}}{\pgfqpoint{1.674cm}{1.101cm}}{\pgfqpoint{1.674cm}{1.065cm}}
\pgfpathcurveto{\pgfqpoint{1.674cm}{1.029cm}}{\pgfqpoint{1.689cm}{0.994cm}}{\pgfqpoint{1.714cm}{0.968cm}}
\pgfpathcurveto{\pgfqpoint{1.74cm}{0.942cm}}{\pgfqpoint{1.775cm}{0.928cm}}{\pgfqpoint{1.811cm}{0.928cm}}
\pgfpathcurveto{\pgfqpoint{1.847cm}{0.928cm}}{\pgfqpoint{1.882cm}{0.942cm}}{\pgfqpoint{1.907cm}{0.968cm}}
\pgfpathcurveto{\pgfqpoint{1.933cm}{0.994cm}}{\pgfqpoint{1.948cm}{1.029cm}}{\pgfqpoint{1.948cm}{1.065cm}}
\pgfusepath{fill}
\begin{pgfscope}
\pgfsetdash{}{0cm}
\pgfsetlinewidth{0.818mm}
\pgfsetmiterlimit{4.0}
\pgfpathmoveto{\pgfqpoint{1.383cm}{0.178cm}}
\pgfpathcurveto{\pgfqpoint{1.383cm}{0.214cm}}{\pgfqpoint{1.369cm}{0.249cm}}{\pgfqpoint{1.343cm}{0.275cm}}
\pgfpathcurveto{\pgfqpoint{1.317cm}{0.3cm}}{\pgfqpoint{1.283cm}{0.315cm}}{\pgfqpoint{1.246cm}{0.315cm}}
\pgfpathcurveto{\pgfqpoint{1.21cm}{0.315cm}}{\pgfqpoint{1.175cm}{0.3cm}}{\pgfqpoint{1.15cm}{0.275cm}}
\pgfpathcurveto{\pgfqpoint{1.124cm}{0.249cm}}{\pgfqpoint{1.11cm}{0.214cm}}{\pgfqpoint{1.11cm}{0.178cm}}
\pgfpathcurveto{\pgfqpoint{1.11cm}{0.141cm}}{\pgfqpoint{1.124cm}{0.107cm}}{\pgfqpoint{1.15cm}{0.081cm}}
\pgfpathcurveto{\pgfqpoint{1.175cm}{0.055cm}}{\pgfqpoint{1.21cm}{0.041cm}}{\pgfqpoint{1.246cm}{0.041cm}}
\pgfpathcurveto{\pgfqpoint{1.283cm}{0.041cm}}{\pgfqpoint{1.317cm}{0.055cm}}{\pgfqpoint{1.343cm}{0.081cm}}
\pgfpathcurveto{\pgfqpoint{1.369cm}{0.107cm}}{\pgfqpoint{1.383cm}{0.141cm}}{\pgfqpoint{1.383cm}{0.178cm}}
\pgfusepath{stroke}
\end{pgfscope}
\end{pgfscope}
\end{pgfscope}
\end{pgfscope}
\end{tikzpicture}}}\|_{\CC^{-1+\alpha}}\lesssim 2^{-(1-\alpha-\kappa)K/2} \|X^{\!\resizebox{!}{.8em}{
\begin{tikzpicture}
\pgfpathmoveto{\pgfqpoint{0cm}{-0.035cm}}
\pgfpathlineto{\pgfqpoint{1.976cm}{-0.035cm}}
\pgfpathlineto{\pgfqpoint{1.976cm}{1.94cm}}
\pgfpathlineto{\pgfqpoint{0cm}{1.94cm}}
\pgfpathclose
\pgfusepath{clip}
\begin{pgfscope}
\begin{pgfscope}
\pgfpathmoveto{\pgfqpoint{0cm}{-0.035cm}}
\pgfpathlineto{\pgfqpoint{1.976cm}{-0.035cm}}
\pgfpathlineto{\pgfqpoint{1.976cm}{1.94cm}}
\pgfpathlineto{\pgfqpoint{0cm}{1.94cm}}
\pgfpathclose
\pgfusepath{clip}
\begin{pgfscope}
\begin{pgfscope}
\pgfsetdash{}{0cm}
\pgfsetlinewidth{0.818mm}
\pgfsetroundcap
\pgfsetroundjoin
\pgfsetmiterlimit{7.0}
\definecolor{eps2pgf_color}{gray}{0}\pgfsetstrokecolor{eps2pgf_color}\pgfsetfillcolor{eps2pgf_color}
\pgfpathmoveto{\pgfqpoint{0.117cm}{1.815cm}}
\pgfpathlineto{\pgfqpoint{0.682cm}{1.065cm}}
\pgfpathlineto{\pgfqpoint{1.246cm}{1.815cm}}
\pgfusepath{stroke}
\end{pgfscope}
\definecolor{eps2pgf_color}{gray}{0}\pgfsetstrokecolor{eps2pgf_color}\pgfsetfillcolor{eps2pgf_color}
\pgfpathmoveto{\pgfqpoint{0.273cm}{1.789cm}}
\pgfpathcurveto{\pgfqpoint{0.273cm}{1.825cm}}{\pgfqpoint{0.259cm}{1.86cm}}{\pgfqpoint{0.233cm}{1.886cm}}
\pgfpathcurveto{\pgfqpoint{0.207cm}{1.912cm}}{\pgfqpoint{0.173cm}{1.926cm}}{\pgfqpoint{0.137cm}{1.926cm}}
\pgfpathcurveto{\pgfqpoint{0.1cm}{1.926cm}}{\pgfqpoint{0.066cm}{1.912cm}}{\pgfqpoint{0.04cm}{1.886cm}}
\pgfpathcurveto{\pgfqpoint{0.014cm}{1.86cm}}{\pgfqpoint{0cm}{1.825cm}}{\pgfqpoint{0cm}{1.789cm}}
\pgfpathcurveto{\pgfqpoint{0cm}{1.753cm}}{\pgfqpoint{0.014cm}{1.718cm}}{\pgfqpoint{0.04cm}{1.692cm}}
\pgfpathcurveto{\pgfqpoint{0.066cm}{1.667cm}}{\pgfqpoint{0.1cm}{1.652cm}}{\pgfqpoint{0.137cm}{1.652cm}}
\pgfpathcurveto{\pgfqpoint{0.173cm}{1.652cm}}{\pgfqpoint{0.207cm}{1.667cm}}{\pgfqpoint{0.233cm}{1.692cm}}
\pgfpathcurveto{\pgfqpoint{0.259cm}{1.718cm}}{\pgfqpoint{0.273cm}{1.753cm}}{\pgfqpoint{0.273cm}{1.789cm}}
\pgfusepath{fill}
\begin{pgfscope}
\pgfsetdash{}{0cm}
\pgfsetlinewidth{0.818mm}
\pgfsetmiterlimit{7.0}
\pgfpathmoveto{\pgfqpoint{0.682cm}{1.065cm}}
\pgfpathlineto{\pgfqpoint{0.679cm}{1.812cm}}
\pgfusepath{stroke}
\end{pgfscope}
\pgfpathmoveto{\pgfqpoint{0.815cm}{1.793cm}}
\pgfpathcurveto{\pgfqpoint{0.815cm}{1.829cm}}{\pgfqpoint{0.801cm}{1.864cm}}{\pgfqpoint{0.775cm}{1.89cm}}
\pgfpathcurveto{\pgfqpoint{0.75cm}{1.915cm}}{\pgfqpoint{0.715cm}{1.93cm}}{\pgfqpoint{0.679cm}{1.93cm}}
\pgfpathcurveto{\pgfqpoint{0.643cm}{1.93cm}}{\pgfqpoint{0.608cm}{1.915cm}}{\pgfqpoint{0.582cm}{1.89cm}}
\pgfpathcurveto{\pgfqpoint{0.557cm}{1.864cm}}{\pgfqpoint{0.542cm}{1.829cm}}{\pgfqpoint{0.542cm}{1.793cm}}
\pgfpathcurveto{\pgfqpoint{0.542cm}{1.756cm}}{\pgfqpoint{0.557cm}{1.722cm}}{\pgfqpoint{0.582cm}{1.696cm}}
\pgfpathcurveto{\pgfqpoint{0.608cm}{1.67cm}}{\pgfqpoint{0.643cm}{1.656cm}}{\pgfqpoint{0.679cm}{1.656cm}}
\pgfpathcurveto{\pgfqpoint{0.715cm}{1.656cm}}{\pgfqpoint{0.75cm}{1.67cm}}{\pgfqpoint{0.775cm}{1.696cm}}
\pgfpathcurveto{\pgfqpoint{0.801cm}{1.722cm}}{\pgfqpoint{0.815cm}{1.756cm}}{\pgfqpoint{0.815cm}{1.793cm}}
\pgfusepath{fill}
\pgfpathmoveto{\pgfqpoint{1.345cm}{1.765cm}}
\pgfpathcurveto{\pgfqpoint{1.345cm}{1.801cm}}{\pgfqpoint{1.331cm}{1.836cm}}{\pgfqpoint{1.305cm}{1.862cm}}
\pgfpathcurveto{\pgfqpoint{1.28cm}{1.887cm}}{\pgfqpoint{1.245cm}{1.902cm}}{\pgfqpoint{1.209cm}{1.902cm}}
\pgfpathcurveto{\pgfqpoint{1.172cm}{1.902cm}}{\pgfqpoint{1.138cm}{1.887cm}}{\pgfqpoint{1.112cm}{1.862cm}}
\pgfpathcurveto{\pgfqpoint{1.087cm}{1.836cm}}{\pgfqpoint{1.072cm}{1.801cm}}{\pgfqpoint{1.072cm}{1.765cm}}
\pgfpathcurveto{\pgfqpoint{1.072cm}{1.728cm}}{\pgfqpoint{1.087cm}{1.694cm}}{\pgfqpoint{1.112cm}{1.668cm}}
\pgfpathcurveto{\pgfqpoint{1.138cm}{1.642cm}}{\pgfqpoint{1.172cm}{1.628cm}}{\pgfqpoint{1.209cm}{1.628cm}}
\pgfpathcurveto{\pgfqpoint{1.245cm}{1.628cm}}{\pgfqpoint{1.28cm}{1.642cm}}{\pgfqpoint{1.305cm}{1.668cm}}
\pgfpathcurveto{\pgfqpoint{1.331cm}{1.694cm}}{\pgfqpoint{1.345cm}{1.728cm}}{\pgfqpoint{1.345cm}{1.765cm}}
\pgfusepath{fill}
\begin{pgfscope}
\pgfsetdash{}{0cm}
\pgfsetlinewidth{0.818mm}
\pgfsetroundcap
\pgfsetroundjoin
\pgfsetmiterlimit{7.0}
\pgfpathmoveto{\pgfqpoint{0.682cm}{1.065cm}}
\pgfpathlineto{\pgfqpoint{1.246cm}{0.315cm}}
\pgfpathlineto{\pgfqpoint{1.811cm}{1.065cm}}
\pgfusepath{stroke}
\end{pgfscope}
\pgfpathmoveto{\pgfqpoint{1.948cm}{1.065cm}}
\pgfpathcurveto{\pgfqpoint{1.948cm}{1.101cm}}{\pgfqpoint{1.933cm}{1.136cm}}{\pgfqpoint{1.907cm}{1.162cm}}
\pgfpathcurveto{\pgfqpoint{1.882cm}{1.187cm}}{\pgfqpoint{1.847cm}{1.202cm}}{\pgfqpoint{1.811cm}{1.202cm}}
\pgfpathcurveto{\pgfqpoint{1.775cm}{1.202cm}}{\pgfqpoint{1.74cm}{1.187cm}}{\pgfqpoint{1.714cm}{1.162cm}}
\pgfpathcurveto{\pgfqpoint{1.689cm}{1.136cm}}{\pgfqpoint{1.674cm}{1.101cm}}{\pgfqpoint{1.674cm}{1.065cm}}
\pgfpathcurveto{\pgfqpoint{1.674cm}{1.029cm}}{\pgfqpoint{1.689cm}{0.994cm}}{\pgfqpoint{1.714cm}{0.968cm}}
\pgfpathcurveto{\pgfqpoint{1.74cm}{0.942cm}}{\pgfqpoint{1.775cm}{0.928cm}}{\pgfqpoint{1.811cm}{0.928cm}}
\pgfpathcurveto{\pgfqpoint{1.847cm}{0.928cm}}{\pgfqpoint{1.882cm}{0.942cm}}{\pgfqpoint{1.907cm}{0.968cm}}
\pgfpathcurveto{\pgfqpoint{1.933cm}{0.994cm}}{\pgfqpoint{1.948cm}{1.029cm}}{\pgfqpoint{1.948cm}{1.065cm}}
\pgfusepath{fill}
\begin{pgfscope}
\pgfsetdash{}{0cm}
\pgfsetlinewidth{0.818mm}
\pgfsetmiterlimit{4.0}
\pgfpathmoveto{\pgfqpoint{1.383cm}{0.178cm}}
\pgfpathcurveto{\pgfqpoint{1.383cm}{0.214cm}}{\pgfqpoint{1.369cm}{0.249cm}}{\pgfqpoint{1.343cm}{0.275cm}}
\pgfpathcurveto{\pgfqpoint{1.317cm}{0.3cm}}{\pgfqpoint{1.283cm}{0.315cm}}{\pgfqpoint{1.246cm}{0.315cm}}
\pgfpathcurveto{\pgfqpoint{1.21cm}{0.315cm}}{\pgfqpoint{1.175cm}{0.3cm}}{\pgfqpoint{1.15cm}{0.275cm}}
\pgfpathcurveto{\pgfqpoint{1.124cm}{0.249cm}}{\pgfqpoint{1.11cm}{0.214cm}}{\pgfqpoint{1.11cm}{0.178cm}}
\pgfpathcurveto{\pgfqpoint{1.11cm}{0.141cm}}{\pgfqpoint{1.124cm}{0.107cm}}{\pgfqpoint{1.15cm}{0.081cm}}
\pgfpathcurveto{\pgfqpoint{1.175cm}{0.055cm}}{\pgfqpoint{1.21cm}{0.041cm}}{\pgfqpoint{1.246cm}{0.041cm}}
\pgfpathcurveto{\pgfqpoint{1.283cm}{0.041cm}}{\pgfqpoint{1.317cm}{0.055cm}}{\pgfqpoint{1.343cm}{0.081cm}}
\pgfpathcurveto{\pgfqpoint{1.369cm}{0.107cm}}{\pgfqpoint{1.383cm}{0.141cm}}{\pgfqpoint{1.383cm}{0.178cm}}
\pgfusepath{stroke}
\end{pgfscope}
\end{pgfscope}
\end{pgfscope}
\end{pgfscope}
\end{tikzpicture}}} \|_{\CC^{-\kappa}(\rho^\sigma)},
$$
$$
\|\UU_> X\|_{\CC^{-1+\alpha}(\rho^\alpha)}\lesssim 2^{-(\frac{1}{2}-\alpha-\kappa)\frac{4}{3}K} \|X\|_{\CC^{-\frac{1}{2}-\kappa}(\rho^\sigma)},
$$
and consequently
\begin{align*}
\|3\UU_>\llbracket X^2 \rrbracket\prec(\phi+\psi)\|_{\CC^{-1+\alpha}(\rho^{2+\alpha})}
&\lesssim \|\phi+\psi\|_{\CC^{\frac{1}{2}+\alpha}(\rho^{\frac{3}{2}+\alpha})}\|\UU_>\llbracket X^2 \rrbracket\|_{\CC^{-\frac{3}{2}}(\rho^{\frac{1}{2}})}\\
&\lesssim 2^{-(\frac{1}{2}-\kappa)K}  \|\phi+\psi\|_{\CC^{\frac{1}{2}+\alpha}(\rho^{\frac{3}{2}+\alpha})}\lesssim 1+\|\psi\|_{L^\infty(\rho)}^{\varepsilon}+ \|\psi\|_{\CC^{\frac{1}{2}+\alpha}(\rho^{\frac{3}{2}+\alpha})}.
\end{align*}
\begin{align*}
\|3( \phi + \psi)\prec\UU_>X^{\!\resizebox{!}{.8em}{
\begin{tikzpicture}
\pgfpathmoveto{\pgfqpoint{0cm}{-0.035cm}}
\pgfpathlineto{\pgfqpoint{1.976cm}{-0.035cm}}
\pgfpathlineto{\pgfqpoint{1.976cm}{1.94cm}}
\pgfpathlineto{\pgfqpoint{0cm}{1.94cm}}
\pgfpathclose
\pgfusepath{clip}
\begin{pgfscope}
\begin{pgfscope}
\pgfpathmoveto{\pgfqpoint{0cm}{-0.035cm}}
\pgfpathlineto{\pgfqpoint{1.976cm}{-0.035cm}}
\pgfpathlineto{\pgfqpoint{1.976cm}{1.94cm}}
\pgfpathlineto{\pgfqpoint{0cm}{1.94cm}}
\pgfpathclose
\pgfusepath{clip}
\begin{pgfscope}
\begin{pgfscope}
\pgfsetdash{}{0cm}
\pgfsetlinewidth{0.818mm}
\pgfsetroundcap
\pgfsetroundjoin
\pgfsetmiterlimit{7.0}
\definecolor{eps2pgf_color}{gray}{0}\pgfsetstrokecolor{eps2pgf_color}\pgfsetfillcolor{eps2pgf_color}
\pgfpathmoveto{\pgfqpoint{0.117cm}{1.815cm}}
\pgfpathlineto{\pgfqpoint{0.682cm}{1.065cm}}
\pgfpathlineto{\pgfqpoint{1.246cm}{1.815cm}}
\pgfusepath{stroke}
\end{pgfscope}
\definecolor{eps2pgf_color}{gray}{0}\pgfsetstrokecolor{eps2pgf_color}\pgfsetfillcolor{eps2pgf_color}
\pgfpathmoveto{\pgfqpoint{0.273cm}{1.789cm}}
\pgfpathcurveto{\pgfqpoint{0.273cm}{1.825cm}}{\pgfqpoint{0.259cm}{1.86cm}}{\pgfqpoint{0.233cm}{1.886cm}}
\pgfpathcurveto{\pgfqpoint{0.207cm}{1.912cm}}{\pgfqpoint{0.173cm}{1.926cm}}{\pgfqpoint{0.137cm}{1.926cm}}
\pgfpathcurveto{\pgfqpoint{0.1cm}{1.926cm}}{\pgfqpoint{0.066cm}{1.912cm}}{\pgfqpoint{0.04cm}{1.886cm}}
\pgfpathcurveto{\pgfqpoint{0.014cm}{1.86cm}}{\pgfqpoint{0cm}{1.825cm}}{\pgfqpoint{0cm}{1.789cm}}
\pgfpathcurveto{\pgfqpoint{0cm}{1.753cm}}{\pgfqpoint{0.014cm}{1.718cm}}{\pgfqpoint{0.04cm}{1.692cm}}
\pgfpathcurveto{\pgfqpoint{0.066cm}{1.667cm}}{\pgfqpoint{0.1cm}{1.652cm}}{\pgfqpoint{0.137cm}{1.652cm}}
\pgfpathcurveto{\pgfqpoint{0.173cm}{1.652cm}}{\pgfqpoint{0.207cm}{1.667cm}}{\pgfqpoint{0.233cm}{1.692cm}}
\pgfpathcurveto{\pgfqpoint{0.259cm}{1.718cm}}{\pgfqpoint{0.273cm}{1.753cm}}{\pgfqpoint{0.273cm}{1.789cm}}
\pgfusepath{fill}
\pgfpathmoveto{\pgfqpoint{1.345cm}{1.765cm}}
\pgfpathcurveto{\pgfqpoint{1.345cm}{1.801cm}}{\pgfqpoint{1.331cm}{1.836cm}}{\pgfqpoint{1.305cm}{1.862cm}}
\pgfpathcurveto{\pgfqpoint{1.28cm}{1.887cm}}{\pgfqpoint{1.245cm}{1.902cm}}{\pgfqpoint{1.209cm}{1.902cm}}
\pgfpathcurveto{\pgfqpoint{1.172cm}{1.902cm}}{\pgfqpoint{1.138cm}{1.887cm}}{\pgfqpoint{1.112cm}{1.862cm}}
\pgfpathcurveto{\pgfqpoint{1.087cm}{1.836cm}}{\pgfqpoint{1.072cm}{1.801cm}}{\pgfqpoint{1.072cm}{1.765cm}}
\pgfpathcurveto{\pgfqpoint{1.072cm}{1.728cm}}{\pgfqpoint{1.087cm}{1.694cm}}{\pgfqpoint{1.112cm}{1.668cm}}
\pgfpathcurveto{\pgfqpoint{1.138cm}{1.642cm}}{\pgfqpoint{1.172cm}{1.628cm}}{\pgfqpoint{1.209cm}{1.628cm}}
\pgfpathcurveto{\pgfqpoint{1.245cm}{1.628cm}}{\pgfqpoint{1.28cm}{1.642cm}}{\pgfqpoint{1.305cm}{1.668cm}}
\pgfpathcurveto{\pgfqpoint{1.331cm}{1.694cm}}{\pgfqpoint{1.345cm}{1.728cm}}{\pgfqpoint{1.345cm}{1.765cm}}
\pgfusepath{fill}
\begin{pgfscope}
\pgfsetdash{}{0cm}
\pgfsetlinewidth{0.818mm}
\pgfsetroundcap
\pgfsetroundjoin
\pgfsetmiterlimit{7.0}
\pgfpathmoveto{\pgfqpoint{0.682cm}{1.065cm}}
\pgfpathlineto{\pgfqpoint{1.246cm}{0.315cm}}
\pgfpathlineto{\pgfqpoint{1.811cm}{1.065cm}}
\pgfusepath{stroke}
\end{pgfscope}
\pgfpathmoveto{\pgfqpoint{1.948cm}{1.065cm}}
\pgfpathcurveto{\pgfqpoint{1.948cm}{1.101cm}}{\pgfqpoint{1.933cm}{1.136cm}}{\pgfqpoint{1.907cm}{1.162cm}}
\pgfpathcurveto{\pgfqpoint{1.882cm}{1.187cm}}{\pgfqpoint{1.847cm}{1.202cm}}{\pgfqpoint{1.811cm}{1.202cm}}
\pgfpathcurveto{\pgfqpoint{1.775cm}{1.202cm}}{\pgfqpoint{1.74cm}{1.187cm}}{\pgfqpoint{1.714cm}{1.162cm}}
\pgfpathcurveto{\pgfqpoint{1.689cm}{1.136cm}}{\pgfqpoint{1.674cm}{1.101cm}}{\pgfqpoint{1.674cm}{1.065cm}}
\pgfpathcurveto{\pgfqpoint{1.674cm}{1.029cm}}{\pgfqpoint{1.689cm}{0.994cm}}{\pgfqpoint{1.714cm}{0.968cm}}
\pgfpathcurveto{\pgfqpoint{1.74cm}{0.942cm}}{\pgfqpoint{1.775cm}{0.928cm}}{\pgfqpoint{1.811cm}{0.928cm}}
\pgfpathcurveto{\pgfqpoint{1.847cm}{0.928cm}}{\pgfqpoint{1.882cm}{0.942cm}}{\pgfqpoint{1.907cm}{0.968cm}}
\pgfpathcurveto{\pgfqpoint{1.933cm}{0.994cm}}{\pgfqpoint{1.948cm}{1.029cm}}{\pgfqpoint{1.948cm}{1.065cm}}
\pgfusepath{fill}
\begin{pgfscope}
\pgfsetdash{}{0cm}
\pgfsetlinewidth{0.818mm}
\pgfsetmiterlimit{7.0}
\pgfpathmoveto{\pgfqpoint{1.246cm}{0.315cm}}
\pgfpathlineto{\pgfqpoint{1.244cm}{1.061cm}}
\pgfusepath{stroke}
\end{pgfscope}
\pgfpathmoveto{\pgfqpoint{1.38cm}{1.065cm}}
\pgfpathcurveto{\pgfqpoint{1.38cm}{1.101cm}}{\pgfqpoint{1.366cm}{1.136cm}}{\pgfqpoint{1.34cm}{1.162cm}}
\pgfpathcurveto{\pgfqpoint{1.315cm}{1.187cm}}{\pgfqpoint{1.28cm}{1.202cm}}{\pgfqpoint{1.244cm}{1.202cm}}
\pgfpathcurveto{\pgfqpoint{1.207cm}{1.202cm}}{\pgfqpoint{1.173cm}{1.187cm}}{\pgfqpoint{1.147cm}{1.162cm}}
\pgfpathcurveto{\pgfqpoint{1.121cm}{1.136cm}}{\pgfqpoint{1.107cm}{1.101cm}}{\pgfqpoint{1.107cm}{1.065cm}}
\pgfpathcurveto{\pgfqpoint{1.107cm}{1.029cm}}{\pgfqpoint{1.121cm}{0.994cm}}{\pgfqpoint{1.147cm}{0.968cm}}
\pgfpathcurveto{\pgfqpoint{1.173cm}{0.942cm}}{\pgfqpoint{1.207cm}{0.928cm}}{\pgfqpoint{1.244cm}{0.928cm}}
\pgfpathcurveto{\pgfqpoint{1.28cm}{0.928cm}}{\pgfqpoint{1.315cm}{0.942cm}}{\pgfqpoint{1.34cm}{0.968cm}}
\pgfpathcurveto{\pgfqpoint{1.366cm}{0.994cm}}{\pgfqpoint{1.38cm}{1.029cm}}{\pgfqpoint{1.38cm}{1.065cm}}
\pgfusepath{fill}
\begin{pgfscope}
\pgfsetdash{}{0cm}
\pgfsetlinewidth{0.818mm}
\pgfsetmiterlimit{4.0}
\pgfpathmoveto{\pgfqpoint{1.383cm}{0.178cm}}
\pgfpathcurveto{\pgfqpoint{1.383cm}{0.214cm}}{\pgfqpoint{1.369cm}{0.249cm}}{\pgfqpoint{1.343cm}{0.275cm}}
\pgfpathcurveto{\pgfqpoint{1.317cm}{0.3cm}}{\pgfqpoint{1.283cm}{0.315cm}}{\pgfqpoint{1.246cm}{0.315cm}}
\pgfpathcurveto{\pgfqpoint{1.21cm}{0.315cm}}{\pgfqpoint{1.175cm}{0.3cm}}{\pgfqpoint{1.15cm}{0.275cm}}
\pgfpathcurveto{\pgfqpoint{1.124cm}{0.249cm}}{\pgfqpoint{1.11cm}{0.214cm}}{\pgfqpoint{1.11cm}{0.178cm}}
\pgfpathcurveto{\pgfqpoint{1.11cm}{0.141cm}}{\pgfqpoint{1.124cm}{0.107cm}}{\pgfqpoint{1.15cm}{0.081cm}}
\pgfpathcurveto{\pgfqpoint{1.175cm}{0.055cm}}{\pgfqpoint{1.21cm}{0.041cm}}{\pgfqpoint{1.246cm}{0.041cm}}
\pgfpathcurveto{\pgfqpoint{1.283cm}{0.041cm}}{\pgfqpoint{1.317cm}{0.055cm}}{\pgfqpoint{1.343cm}{0.081cm}}
\pgfpathcurveto{\pgfqpoint{1.369cm}{0.107cm}}{\pgfqpoint{1.383cm}{0.141cm}}{\pgfqpoint{1.383cm}{0.178cm}}
\pgfusepath{stroke}
\end{pgfscope}
\end{pgfscope}
\end{pgfscope}
\end{pgfscope}
\end{tikzpicture}}}\|_{\CC^{-1+\alpha}(\rho)}
&\lesssim  \|\phi+\psi\|_{L^\infty(\rho)}\|\UU_>X^{\!\resizebox{!}{.8em}{
\begin{tikzpicture}
\pgfpathmoveto{\pgfqpoint{0cm}{-0.035cm}}
\pgfpathlineto{\pgfqpoint{1.976cm}{-0.035cm}}
\pgfpathlineto{\pgfqpoint{1.976cm}{1.94cm}}
\pgfpathlineto{\pgfqpoint{0cm}{1.94cm}}
\pgfpathclose
\pgfusepath{clip}
\begin{pgfscope}
\begin{pgfscope}
\pgfpathmoveto{\pgfqpoint{0cm}{-0.035cm}}
\pgfpathlineto{\pgfqpoint{1.976cm}{-0.035cm}}
\pgfpathlineto{\pgfqpoint{1.976cm}{1.94cm}}
\pgfpathlineto{\pgfqpoint{0cm}{1.94cm}}
\pgfpathclose
\pgfusepath{clip}
\begin{pgfscope}
\begin{pgfscope}
\pgfsetdash{}{0cm}
\pgfsetlinewidth{0.818mm}
\pgfsetroundcap
\pgfsetroundjoin
\pgfsetmiterlimit{7.0}
\definecolor{eps2pgf_color}{gray}{0}\pgfsetstrokecolor{eps2pgf_color}\pgfsetfillcolor{eps2pgf_color}
\pgfpathmoveto{\pgfqpoint{0.117cm}{1.815cm}}
\pgfpathlineto{\pgfqpoint{0.682cm}{1.065cm}}
\pgfpathlineto{\pgfqpoint{1.246cm}{1.815cm}}
\pgfusepath{stroke}
\end{pgfscope}
\definecolor{eps2pgf_color}{gray}{0}\pgfsetstrokecolor{eps2pgf_color}\pgfsetfillcolor{eps2pgf_color}
\pgfpathmoveto{\pgfqpoint{0.273cm}{1.789cm}}
\pgfpathcurveto{\pgfqpoint{0.273cm}{1.825cm}}{\pgfqpoint{0.259cm}{1.86cm}}{\pgfqpoint{0.233cm}{1.886cm}}
\pgfpathcurveto{\pgfqpoint{0.207cm}{1.912cm}}{\pgfqpoint{0.173cm}{1.926cm}}{\pgfqpoint{0.137cm}{1.926cm}}
\pgfpathcurveto{\pgfqpoint{0.1cm}{1.926cm}}{\pgfqpoint{0.066cm}{1.912cm}}{\pgfqpoint{0.04cm}{1.886cm}}
\pgfpathcurveto{\pgfqpoint{0.014cm}{1.86cm}}{\pgfqpoint{0cm}{1.825cm}}{\pgfqpoint{0cm}{1.789cm}}
\pgfpathcurveto{\pgfqpoint{0cm}{1.753cm}}{\pgfqpoint{0.014cm}{1.718cm}}{\pgfqpoint{0.04cm}{1.692cm}}
\pgfpathcurveto{\pgfqpoint{0.066cm}{1.667cm}}{\pgfqpoint{0.1cm}{1.652cm}}{\pgfqpoint{0.137cm}{1.652cm}}
\pgfpathcurveto{\pgfqpoint{0.173cm}{1.652cm}}{\pgfqpoint{0.207cm}{1.667cm}}{\pgfqpoint{0.233cm}{1.692cm}}
\pgfpathcurveto{\pgfqpoint{0.259cm}{1.718cm}}{\pgfqpoint{0.273cm}{1.753cm}}{\pgfqpoint{0.273cm}{1.789cm}}
\pgfusepath{fill}
\pgfpathmoveto{\pgfqpoint{1.345cm}{1.765cm}}
\pgfpathcurveto{\pgfqpoint{1.345cm}{1.801cm}}{\pgfqpoint{1.331cm}{1.836cm}}{\pgfqpoint{1.305cm}{1.862cm}}
\pgfpathcurveto{\pgfqpoint{1.28cm}{1.887cm}}{\pgfqpoint{1.245cm}{1.902cm}}{\pgfqpoint{1.209cm}{1.902cm}}
\pgfpathcurveto{\pgfqpoint{1.172cm}{1.902cm}}{\pgfqpoint{1.138cm}{1.887cm}}{\pgfqpoint{1.112cm}{1.862cm}}
\pgfpathcurveto{\pgfqpoint{1.087cm}{1.836cm}}{\pgfqpoint{1.072cm}{1.801cm}}{\pgfqpoint{1.072cm}{1.765cm}}
\pgfpathcurveto{\pgfqpoint{1.072cm}{1.728cm}}{\pgfqpoint{1.087cm}{1.694cm}}{\pgfqpoint{1.112cm}{1.668cm}}
\pgfpathcurveto{\pgfqpoint{1.138cm}{1.642cm}}{\pgfqpoint{1.172cm}{1.628cm}}{\pgfqpoint{1.209cm}{1.628cm}}
\pgfpathcurveto{\pgfqpoint{1.245cm}{1.628cm}}{\pgfqpoint{1.28cm}{1.642cm}}{\pgfqpoint{1.305cm}{1.668cm}}
\pgfpathcurveto{\pgfqpoint{1.331cm}{1.694cm}}{\pgfqpoint{1.345cm}{1.728cm}}{\pgfqpoint{1.345cm}{1.765cm}}
\pgfusepath{fill}
\begin{pgfscope}
\pgfsetdash{}{0cm}
\pgfsetlinewidth{0.818mm}
\pgfsetroundcap
\pgfsetroundjoin
\pgfsetmiterlimit{7.0}
\pgfpathmoveto{\pgfqpoint{0.682cm}{1.065cm}}
\pgfpathlineto{\pgfqpoint{1.246cm}{0.315cm}}
\pgfpathlineto{\pgfqpoint{1.811cm}{1.065cm}}
\pgfusepath{stroke}
\end{pgfscope}
\pgfpathmoveto{\pgfqpoint{1.948cm}{1.065cm}}
\pgfpathcurveto{\pgfqpoint{1.948cm}{1.101cm}}{\pgfqpoint{1.933cm}{1.136cm}}{\pgfqpoint{1.907cm}{1.162cm}}
\pgfpathcurveto{\pgfqpoint{1.882cm}{1.187cm}}{\pgfqpoint{1.847cm}{1.202cm}}{\pgfqpoint{1.811cm}{1.202cm}}
\pgfpathcurveto{\pgfqpoint{1.775cm}{1.202cm}}{\pgfqpoint{1.74cm}{1.187cm}}{\pgfqpoint{1.714cm}{1.162cm}}
\pgfpathcurveto{\pgfqpoint{1.689cm}{1.136cm}}{\pgfqpoint{1.674cm}{1.101cm}}{\pgfqpoint{1.674cm}{1.065cm}}
\pgfpathcurveto{\pgfqpoint{1.674cm}{1.029cm}}{\pgfqpoint{1.689cm}{0.994cm}}{\pgfqpoint{1.714cm}{0.968cm}}
\pgfpathcurveto{\pgfqpoint{1.74cm}{0.942cm}}{\pgfqpoint{1.775cm}{0.928cm}}{\pgfqpoint{1.811cm}{0.928cm}}
\pgfpathcurveto{\pgfqpoint{1.847cm}{0.928cm}}{\pgfqpoint{1.882cm}{0.942cm}}{\pgfqpoint{1.907cm}{0.968cm}}
\pgfpathcurveto{\pgfqpoint{1.933cm}{0.994cm}}{\pgfqpoint{1.948cm}{1.029cm}}{\pgfqpoint{1.948cm}{1.065cm}}
\pgfusepath{fill}
\begin{pgfscope}
\pgfsetdash{}{0cm}
\pgfsetlinewidth{0.818mm}
\pgfsetmiterlimit{7.0}
\pgfpathmoveto{\pgfqpoint{1.246cm}{0.315cm}}
\pgfpathlineto{\pgfqpoint{1.244cm}{1.061cm}}
\pgfusepath{stroke}
\end{pgfscope}
\pgfpathmoveto{\pgfqpoint{1.38cm}{1.065cm}}
\pgfpathcurveto{\pgfqpoint{1.38cm}{1.101cm}}{\pgfqpoint{1.366cm}{1.136cm}}{\pgfqpoint{1.34cm}{1.162cm}}
\pgfpathcurveto{\pgfqpoint{1.315cm}{1.187cm}}{\pgfqpoint{1.28cm}{1.202cm}}{\pgfqpoint{1.244cm}{1.202cm}}
\pgfpathcurveto{\pgfqpoint{1.207cm}{1.202cm}}{\pgfqpoint{1.173cm}{1.187cm}}{\pgfqpoint{1.147cm}{1.162cm}}
\pgfpathcurveto{\pgfqpoint{1.121cm}{1.136cm}}{\pgfqpoint{1.107cm}{1.101cm}}{\pgfqpoint{1.107cm}{1.065cm}}
\pgfpathcurveto{\pgfqpoint{1.107cm}{1.029cm}}{\pgfqpoint{1.121cm}{0.994cm}}{\pgfqpoint{1.147cm}{0.968cm}}
\pgfpathcurveto{\pgfqpoint{1.173cm}{0.942cm}}{\pgfqpoint{1.207cm}{0.928cm}}{\pgfqpoint{1.244cm}{0.928cm}}
\pgfpathcurveto{\pgfqpoint{1.28cm}{0.928cm}}{\pgfqpoint{1.315cm}{0.942cm}}{\pgfqpoint{1.34cm}{0.968cm}}
\pgfpathcurveto{\pgfqpoint{1.366cm}{0.994cm}}{\pgfqpoint{1.38cm}{1.029cm}}{\pgfqpoint{1.38cm}{1.065cm}}
\pgfusepath{fill}
\begin{pgfscope}
\pgfsetdash{}{0cm}
\pgfsetlinewidth{0.818mm}
\pgfsetmiterlimit{4.0}
\pgfpathmoveto{\pgfqpoint{1.383cm}{0.178cm}}
\pgfpathcurveto{\pgfqpoint{1.383cm}{0.214cm}}{\pgfqpoint{1.369cm}{0.249cm}}{\pgfqpoint{1.343cm}{0.275cm}}
\pgfpathcurveto{\pgfqpoint{1.317cm}{0.3cm}}{\pgfqpoint{1.283cm}{0.315cm}}{\pgfqpoint{1.246cm}{0.315cm}}
\pgfpathcurveto{\pgfqpoint{1.21cm}{0.315cm}}{\pgfqpoint{1.175cm}{0.3cm}}{\pgfqpoint{1.15cm}{0.275cm}}
\pgfpathcurveto{\pgfqpoint{1.124cm}{0.249cm}}{\pgfqpoint{1.11cm}{0.214cm}}{\pgfqpoint{1.11cm}{0.178cm}}
\pgfpathcurveto{\pgfqpoint{1.11cm}{0.141cm}}{\pgfqpoint{1.124cm}{0.107cm}}{\pgfqpoint{1.15cm}{0.081cm}}
\pgfpathcurveto{\pgfqpoint{1.175cm}{0.055cm}}{\pgfqpoint{1.21cm}{0.041cm}}{\pgfqpoint{1.246cm}{0.041cm}}
\pgfpathcurveto{\pgfqpoint{1.283cm}{0.041cm}}{\pgfqpoint{1.317cm}{0.055cm}}{\pgfqpoint{1.343cm}{0.081cm}}
\pgfpathcurveto{\pgfqpoint{1.369cm}{0.107cm}}{\pgfqpoint{1.383cm}{0.141cm}}{\pgfqpoint{1.383cm}{0.178cm}}
\pgfusepath{stroke}
\end{pgfscope}
\end{pgfscope}
\end{pgfscope}
\end{pgfscope}
\end{tikzpicture}}}\|_{\CC^{-1+\alpha}}\lesssim 2^{-(1-\alpha-\kappa)K/2}\|\phi+\psi\|_{L^\infty(\rho)}\\
&\lesssim 1+\|\psi\|_{L^\infty(\rho)}^{\varepsilon},
\end{align*}
\begin{align*}
&\|6\UU_>X\succ(X^{\!\resizebox{0.6em}{!}{
\begin{tikzpicture}
\pgfpathmoveto{\pgfqpoint{0cm}{-0.035cm}}
\pgfpathlineto{\pgfqpoint{1.376cm}{-0.035cm}}
\pgfpathlineto{\pgfqpoint{1.376cm}{1.552cm}}
\pgfpathlineto{\pgfqpoint{0cm}{1.552cm}}
\pgfpathclose
\pgfusepath{clip}
\begin{pgfscope}
\begin{pgfscope}
\pgfpathmoveto{\pgfqpoint{0cm}{-0.035cm}}
\pgfpathlineto{\pgfqpoint{1.376cm}{-0.035cm}}
\pgfpathlineto{\pgfqpoint{1.376cm}{1.552cm}}
\pgfpathlineto{\pgfqpoint{0cm}{1.552cm}}
\pgfpathclose
\pgfusepath{clip}
\begin{pgfscope}
\begin{pgfscope}
\pgfsetdash{}{0cm}
\pgfsetlinewidth{0.818mm}
\pgfsetroundcap
\pgfsetroundjoin
\pgfsetmiterlimit{7.0}
\definecolor{eps2pgf_color}{gray}{0}\pgfsetstrokecolor{eps2pgf_color}\pgfsetfillcolor{eps2pgf_color}
\pgfpathmoveto{\pgfqpoint{0.117cm}{1.421cm}}
\pgfpathlineto{\pgfqpoint{0.682cm}{0.671cm}}
\pgfpathlineto{\pgfqpoint{1.246cm}{1.421cm}}
\pgfusepath{stroke}
\end{pgfscope}
\definecolor{eps2pgf_color}{gray}{0}\pgfsetstrokecolor{eps2pgf_color}\pgfsetfillcolor{eps2pgf_color}
\pgfpathmoveto{\pgfqpoint{0.273cm}{1.395cm}}
\pgfpathcurveto{\pgfqpoint{0.273cm}{1.432cm}}{\pgfqpoint{0.259cm}{1.467cm}}{\pgfqpoint{0.233cm}{1.492cm}}
\pgfpathcurveto{\pgfqpoint{0.207cm}{1.518cm}}{\pgfqpoint{0.173cm}{1.532cm}}{\pgfqpoint{0.137cm}{1.532cm}}
\pgfpathcurveto{\pgfqpoint{0.1cm}{1.532cm}}{\pgfqpoint{0.066cm}{1.518cm}}{\pgfqpoint{0.04cm}{1.492cm}}
\pgfpathcurveto{\pgfqpoint{0.014cm}{1.467cm}}{\pgfqpoint{0cm}{1.432cm}}{\pgfqpoint{0cm}{1.395cm}}
\pgfpathcurveto{\pgfqpoint{0cm}{1.359cm}}{\pgfqpoint{0.014cm}{1.324cm}}{\pgfqpoint{0.04cm}{1.299cm}}
\pgfpathcurveto{\pgfqpoint{0.066cm}{1.273cm}}{\pgfqpoint{0.1cm}{1.258cm}}{\pgfqpoint{0.137cm}{1.258cm}}
\pgfpathcurveto{\pgfqpoint{0.173cm}{1.258cm}}{\pgfqpoint{0.207cm}{1.273cm}}{\pgfqpoint{0.233cm}{1.299cm}}
\pgfpathcurveto{\pgfqpoint{0.259cm}{1.324cm}}{\pgfqpoint{0.273cm}{1.359cm}}{\pgfqpoint{0.273cm}{1.395cm}}
\pgfusepath{fill}
\begin{pgfscope}
\pgfsetdash{}{0cm}
\pgfsetlinewidth{0.818mm}
\pgfsetmiterlimit{7.0}
\pgfpathmoveto{\pgfqpoint{0.682cm}{0.671cm}}
\pgfpathlineto{\pgfqpoint{0.679cm}{1.418cm}}
\pgfusepath{stroke}
\end{pgfscope}
\pgfpathmoveto{\pgfqpoint{0.815cm}{1.399cm}}
\pgfpathcurveto{\pgfqpoint{0.815cm}{1.435cm}}{\pgfqpoint{0.801cm}{1.47cm}}{\pgfqpoint{0.775cm}{1.496cm}}
\pgfpathcurveto{\pgfqpoint{0.75cm}{1.521cm}}{\pgfqpoint{0.715cm}{1.536cm}}{\pgfqpoint{0.679cm}{1.536cm}}
\pgfpathcurveto{\pgfqpoint{0.643cm}{1.536cm}}{\pgfqpoint{0.608cm}{1.521cm}}{\pgfqpoint{0.582cm}{1.496cm}}
\pgfpathcurveto{\pgfqpoint{0.557cm}{1.47cm}}{\pgfqpoint{0.542cm}{1.435cm}}{\pgfqpoint{0.542cm}{1.399cm}}
\pgfpathcurveto{\pgfqpoint{0.542cm}{1.363cm}}{\pgfqpoint{0.557cm}{1.328cm}}{\pgfqpoint{0.582cm}{1.302cm}}
\pgfpathcurveto{\pgfqpoint{0.608cm}{1.276cm}}{\pgfqpoint{0.643cm}{1.262cm}}{\pgfqpoint{0.679cm}{1.262cm}}
\pgfpathcurveto{\pgfqpoint{0.715cm}{1.262cm}}{\pgfqpoint{0.75cm}{1.276cm}}{\pgfqpoint{0.775cm}{1.302cm}}
\pgfpathcurveto{\pgfqpoint{0.801cm}{1.328cm}}{\pgfqpoint{0.815cm}{1.363cm}}{\pgfqpoint{0.815cm}{1.399cm}}
\pgfusepath{fill}
\pgfpathmoveto{\pgfqpoint{1.345cm}{1.371cm}}
\pgfpathcurveto{\pgfqpoint{1.345cm}{1.408cm}}{\pgfqpoint{1.331cm}{1.442cm}}{\pgfqpoint{1.305cm}{1.468cm}}
\pgfpathcurveto{\pgfqpoint{1.28cm}{1.494cm}}{\pgfqpoint{1.245cm}{1.508cm}}{\pgfqpoint{1.209cm}{1.508cm}}
\pgfpathcurveto{\pgfqpoint{1.172cm}{1.508cm}}{\pgfqpoint{1.138cm}{1.494cm}}{\pgfqpoint{1.112cm}{1.468cm}}
\pgfpathcurveto{\pgfqpoint{1.087cm}{1.442cm}}{\pgfqpoint{1.072cm}{1.408cm}}{\pgfqpoint{1.072cm}{1.371cm}}
\pgfpathcurveto{\pgfqpoint{1.072cm}{1.335cm}}{\pgfqpoint{1.087cm}{1.3cm}}{\pgfqpoint{1.112cm}{1.274cm}}
\pgfpathcurveto{\pgfqpoint{1.138cm}{1.249cm}}{\pgfqpoint{1.172cm}{1.234cm}}{\pgfqpoint{1.209cm}{1.234cm}}
\pgfpathcurveto{\pgfqpoint{1.245cm}{1.234cm}}{\pgfqpoint{1.28cm}{1.249cm}}{\pgfqpoint{1.305cm}{1.274cm}}
\pgfpathcurveto{\pgfqpoint{1.331cm}{1.3cm}}{\pgfqpoint{1.345cm}{1.335cm}}{\pgfqpoint{1.345cm}{1.371cm}}
\pgfusepath{fill}
\begin{pgfscope}
\pgfsetdash{}{0cm}
\pgfsetlinewidth{0.818mm}
\pgfsetroundcap
\pgfsetmiterlimit{4.0}
\pgfpathmoveto{\pgfqpoint{0.682cm}{0.671cm}}
\pgfpathlineto{\pgfqpoint{0.682cm}{0.042cm}}
\pgfusepath{stroke}
\end{pgfscope}
\end{pgfscope}
\end{pgfscope}
\end{pgfscope}
\end{tikzpicture}}}(\phi+\psi))\|_{\CC^{-1+\alpha}(\rho)}+\|6\UU_> X\prec(X^{\!\resizebox{0.6em}{!}{
\begin{tikzpicture}
\pgfpathmoveto{\pgfqpoint{0cm}{-0.035cm}}
\pgfpathlineto{\pgfqpoint{1.376cm}{-0.035cm}}
\pgfpathlineto{\pgfqpoint{1.376cm}{1.552cm}}
\pgfpathlineto{\pgfqpoint{0cm}{1.552cm}}
\pgfpathclose
\pgfusepath{clip}
\begin{pgfscope}
\begin{pgfscope}
\pgfpathmoveto{\pgfqpoint{0cm}{-0.035cm}}
\pgfpathlineto{\pgfqpoint{1.376cm}{-0.035cm}}
\pgfpathlineto{\pgfqpoint{1.376cm}{1.552cm}}
\pgfpathlineto{\pgfqpoint{0cm}{1.552cm}}
\pgfpathclose
\pgfusepath{clip}
\begin{pgfscope}
\begin{pgfscope}
\pgfsetdash{}{0cm}
\pgfsetlinewidth{0.818mm}
\pgfsetroundcap
\pgfsetroundjoin
\pgfsetmiterlimit{7.0}
\definecolor{eps2pgf_color}{gray}{0}\pgfsetstrokecolor{eps2pgf_color}\pgfsetfillcolor{eps2pgf_color}
\pgfpathmoveto{\pgfqpoint{0.117cm}{1.421cm}}
\pgfpathlineto{\pgfqpoint{0.682cm}{0.671cm}}
\pgfpathlineto{\pgfqpoint{1.246cm}{1.421cm}}
\pgfusepath{stroke}
\end{pgfscope}
\definecolor{eps2pgf_color}{gray}{0}\pgfsetstrokecolor{eps2pgf_color}\pgfsetfillcolor{eps2pgf_color}
\pgfpathmoveto{\pgfqpoint{0.273cm}{1.395cm}}
\pgfpathcurveto{\pgfqpoint{0.273cm}{1.432cm}}{\pgfqpoint{0.259cm}{1.467cm}}{\pgfqpoint{0.233cm}{1.492cm}}
\pgfpathcurveto{\pgfqpoint{0.207cm}{1.518cm}}{\pgfqpoint{0.173cm}{1.532cm}}{\pgfqpoint{0.137cm}{1.532cm}}
\pgfpathcurveto{\pgfqpoint{0.1cm}{1.532cm}}{\pgfqpoint{0.066cm}{1.518cm}}{\pgfqpoint{0.04cm}{1.492cm}}
\pgfpathcurveto{\pgfqpoint{0.014cm}{1.467cm}}{\pgfqpoint{0cm}{1.432cm}}{\pgfqpoint{0cm}{1.395cm}}
\pgfpathcurveto{\pgfqpoint{0cm}{1.359cm}}{\pgfqpoint{0.014cm}{1.324cm}}{\pgfqpoint{0.04cm}{1.299cm}}
\pgfpathcurveto{\pgfqpoint{0.066cm}{1.273cm}}{\pgfqpoint{0.1cm}{1.258cm}}{\pgfqpoint{0.137cm}{1.258cm}}
\pgfpathcurveto{\pgfqpoint{0.173cm}{1.258cm}}{\pgfqpoint{0.207cm}{1.273cm}}{\pgfqpoint{0.233cm}{1.299cm}}
\pgfpathcurveto{\pgfqpoint{0.259cm}{1.324cm}}{\pgfqpoint{0.273cm}{1.359cm}}{\pgfqpoint{0.273cm}{1.395cm}}
\pgfusepath{fill}
\begin{pgfscope}
\pgfsetdash{}{0cm}
\pgfsetlinewidth{0.818mm}
\pgfsetmiterlimit{7.0}
\pgfpathmoveto{\pgfqpoint{0.682cm}{0.671cm}}
\pgfpathlineto{\pgfqpoint{0.679cm}{1.418cm}}
\pgfusepath{stroke}
\end{pgfscope}
\pgfpathmoveto{\pgfqpoint{0.815cm}{1.399cm}}
\pgfpathcurveto{\pgfqpoint{0.815cm}{1.435cm}}{\pgfqpoint{0.801cm}{1.47cm}}{\pgfqpoint{0.775cm}{1.496cm}}
\pgfpathcurveto{\pgfqpoint{0.75cm}{1.521cm}}{\pgfqpoint{0.715cm}{1.536cm}}{\pgfqpoint{0.679cm}{1.536cm}}
\pgfpathcurveto{\pgfqpoint{0.643cm}{1.536cm}}{\pgfqpoint{0.608cm}{1.521cm}}{\pgfqpoint{0.582cm}{1.496cm}}
\pgfpathcurveto{\pgfqpoint{0.557cm}{1.47cm}}{\pgfqpoint{0.542cm}{1.435cm}}{\pgfqpoint{0.542cm}{1.399cm}}
\pgfpathcurveto{\pgfqpoint{0.542cm}{1.363cm}}{\pgfqpoint{0.557cm}{1.328cm}}{\pgfqpoint{0.582cm}{1.302cm}}
\pgfpathcurveto{\pgfqpoint{0.608cm}{1.276cm}}{\pgfqpoint{0.643cm}{1.262cm}}{\pgfqpoint{0.679cm}{1.262cm}}
\pgfpathcurveto{\pgfqpoint{0.715cm}{1.262cm}}{\pgfqpoint{0.75cm}{1.276cm}}{\pgfqpoint{0.775cm}{1.302cm}}
\pgfpathcurveto{\pgfqpoint{0.801cm}{1.328cm}}{\pgfqpoint{0.815cm}{1.363cm}}{\pgfqpoint{0.815cm}{1.399cm}}
\pgfusepath{fill}
\pgfpathmoveto{\pgfqpoint{1.345cm}{1.371cm}}
\pgfpathcurveto{\pgfqpoint{1.345cm}{1.408cm}}{\pgfqpoint{1.331cm}{1.442cm}}{\pgfqpoint{1.305cm}{1.468cm}}
\pgfpathcurveto{\pgfqpoint{1.28cm}{1.494cm}}{\pgfqpoint{1.245cm}{1.508cm}}{\pgfqpoint{1.209cm}{1.508cm}}
\pgfpathcurveto{\pgfqpoint{1.172cm}{1.508cm}}{\pgfqpoint{1.138cm}{1.494cm}}{\pgfqpoint{1.112cm}{1.468cm}}
\pgfpathcurveto{\pgfqpoint{1.087cm}{1.442cm}}{\pgfqpoint{1.072cm}{1.408cm}}{\pgfqpoint{1.072cm}{1.371cm}}
\pgfpathcurveto{\pgfqpoint{1.072cm}{1.335cm}}{\pgfqpoint{1.087cm}{1.3cm}}{\pgfqpoint{1.112cm}{1.274cm}}
\pgfpathcurveto{\pgfqpoint{1.138cm}{1.249cm}}{\pgfqpoint{1.172cm}{1.234cm}}{\pgfqpoint{1.209cm}{1.234cm}}
\pgfpathcurveto{\pgfqpoint{1.245cm}{1.234cm}}{\pgfqpoint{1.28cm}{1.249cm}}{\pgfqpoint{1.305cm}{1.274cm}}
\pgfpathcurveto{\pgfqpoint{1.331cm}{1.3cm}}{\pgfqpoint{1.345cm}{1.335cm}}{\pgfqpoint{1.345cm}{1.371cm}}
\pgfusepath{fill}
\begin{pgfscope}
\pgfsetdash{}{0cm}
\pgfsetlinewidth{0.818mm}
\pgfsetroundcap
\pgfsetmiterlimit{4.0}
\pgfpathmoveto{\pgfqpoint{0.682cm}{0.671cm}}
\pgfpathlineto{\pgfqpoint{0.682cm}{0.042cm}}
\pgfusepath{stroke}
\end{pgfscope}
\end{pgfscope}
\end{pgfscope}
\end{pgfscope}
\end{tikzpicture}}}(\phi+\psi))\|_{\CC^{-1+\alpha}(\rho)}\\
&\quad\lesssim \|\phi+\psi\|_{L^\infty(\rho)}\|\UU_>X\|_{\CC^{-1+\alpha}(\rho^{-\alpha})}\\
&\quad\lesssim 2^{-(\frac{1}{2}-\alpha-\kappa)\frac{2}{3}K}\|\phi+\psi\|_{L^\infty(\rho)}\lesssim 1+\|\psi\|_{L^\infty(\rho)}^{\varepsilon},
\end{align*}
\begin{align*}
\|6(\phi+\psi)\prec\UU_{>}X^{\!\resizebox{!}{.8em}{
\begin{tikzpicture}
\pgfpathmoveto{\pgfqpoint{0cm}{-0.035cm}}
\pgfpathlineto{\pgfqpoint{1.976cm}{-0.035cm}}
\pgfpathlineto{\pgfqpoint{1.976cm}{1.94cm}}
\pgfpathlineto{\pgfqpoint{0cm}{1.94cm}}
\pgfpathclose
\pgfusepath{clip}
\begin{pgfscope}
\begin{pgfscope}
\pgfpathmoveto{\pgfqpoint{0cm}{-0.035cm}}
\pgfpathlineto{\pgfqpoint{1.976cm}{-0.035cm}}
\pgfpathlineto{\pgfqpoint{1.976cm}{1.94cm}}
\pgfpathlineto{\pgfqpoint{0cm}{1.94cm}}
\pgfpathclose
\pgfusepath{clip}
\begin{pgfscope}
\begin{pgfscope}
\pgfsetdash{}{0cm}
\pgfsetlinewidth{0.818mm}
\pgfsetroundcap
\pgfsetroundjoin
\pgfsetmiterlimit{7.0}
\definecolor{eps2pgf_color}{gray}{0}\pgfsetstrokecolor{eps2pgf_color}\pgfsetfillcolor{eps2pgf_color}
\pgfpathmoveto{\pgfqpoint{0.117cm}{1.815cm}}
\pgfpathlineto{\pgfqpoint{0.682cm}{1.065cm}}
\pgfpathlineto{\pgfqpoint{1.246cm}{1.815cm}}
\pgfusepath{stroke}
\end{pgfscope}
\definecolor{eps2pgf_color}{gray}{0}\pgfsetstrokecolor{eps2pgf_color}\pgfsetfillcolor{eps2pgf_color}
\pgfpathmoveto{\pgfqpoint{0.273cm}{1.789cm}}
\pgfpathcurveto{\pgfqpoint{0.273cm}{1.825cm}}{\pgfqpoint{0.259cm}{1.86cm}}{\pgfqpoint{0.233cm}{1.886cm}}
\pgfpathcurveto{\pgfqpoint{0.207cm}{1.912cm}}{\pgfqpoint{0.173cm}{1.926cm}}{\pgfqpoint{0.137cm}{1.926cm}}
\pgfpathcurveto{\pgfqpoint{0.1cm}{1.926cm}}{\pgfqpoint{0.066cm}{1.912cm}}{\pgfqpoint{0.04cm}{1.886cm}}
\pgfpathcurveto{\pgfqpoint{0.014cm}{1.86cm}}{\pgfqpoint{0cm}{1.825cm}}{\pgfqpoint{0cm}{1.789cm}}
\pgfpathcurveto{\pgfqpoint{0cm}{1.753cm}}{\pgfqpoint{0.014cm}{1.718cm}}{\pgfqpoint{0.04cm}{1.692cm}}
\pgfpathcurveto{\pgfqpoint{0.066cm}{1.667cm}}{\pgfqpoint{0.1cm}{1.652cm}}{\pgfqpoint{0.137cm}{1.652cm}}
\pgfpathcurveto{\pgfqpoint{0.173cm}{1.652cm}}{\pgfqpoint{0.207cm}{1.667cm}}{\pgfqpoint{0.233cm}{1.692cm}}
\pgfpathcurveto{\pgfqpoint{0.259cm}{1.718cm}}{\pgfqpoint{0.273cm}{1.753cm}}{\pgfqpoint{0.273cm}{1.789cm}}
\pgfusepath{fill}
\begin{pgfscope}
\pgfsetdash{}{0cm}
\pgfsetlinewidth{0.818mm}
\pgfsetmiterlimit{7.0}
\pgfpathmoveto{\pgfqpoint{0.682cm}{1.065cm}}
\pgfpathlineto{\pgfqpoint{0.679cm}{1.812cm}}
\pgfusepath{stroke}
\end{pgfscope}
\pgfpathmoveto{\pgfqpoint{0.815cm}{1.793cm}}
\pgfpathcurveto{\pgfqpoint{0.815cm}{1.829cm}}{\pgfqpoint{0.801cm}{1.864cm}}{\pgfqpoint{0.775cm}{1.89cm}}
\pgfpathcurveto{\pgfqpoint{0.75cm}{1.915cm}}{\pgfqpoint{0.715cm}{1.93cm}}{\pgfqpoint{0.679cm}{1.93cm}}
\pgfpathcurveto{\pgfqpoint{0.643cm}{1.93cm}}{\pgfqpoint{0.608cm}{1.915cm}}{\pgfqpoint{0.582cm}{1.89cm}}
\pgfpathcurveto{\pgfqpoint{0.557cm}{1.864cm}}{\pgfqpoint{0.542cm}{1.829cm}}{\pgfqpoint{0.542cm}{1.793cm}}
\pgfpathcurveto{\pgfqpoint{0.542cm}{1.756cm}}{\pgfqpoint{0.557cm}{1.722cm}}{\pgfqpoint{0.582cm}{1.696cm}}
\pgfpathcurveto{\pgfqpoint{0.608cm}{1.67cm}}{\pgfqpoint{0.643cm}{1.656cm}}{\pgfqpoint{0.679cm}{1.656cm}}
\pgfpathcurveto{\pgfqpoint{0.715cm}{1.656cm}}{\pgfqpoint{0.75cm}{1.67cm}}{\pgfqpoint{0.775cm}{1.696cm}}
\pgfpathcurveto{\pgfqpoint{0.801cm}{1.722cm}}{\pgfqpoint{0.815cm}{1.756cm}}{\pgfqpoint{0.815cm}{1.793cm}}
\pgfusepath{fill}
\pgfpathmoveto{\pgfqpoint{1.345cm}{1.765cm}}
\pgfpathcurveto{\pgfqpoint{1.345cm}{1.801cm}}{\pgfqpoint{1.331cm}{1.836cm}}{\pgfqpoint{1.305cm}{1.862cm}}
\pgfpathcurveto{\pgfqpoint{1.28cm}{1.887cm}}{\pgfqpoint{1.245cm}{1.902cm}}{\pgfqpoint{1.209cm}{1.902cm}}
\pgfpathcurveto{\pgfqpoint{1.172cm}{1.902cm}}{\pgfqpoint{1.138cm}{1.887cm}}{\pgfqpoint{1.112cm}{1.862cm}}
\pgfpathcurveto{\pgfqpoint{1.087cm}{1.836cm}}{\pgfqpoint{1.072cm}{1.801cm}}{\pgfqpoint{1.072cm}{1.765cm}}
\pgfpathcurveto{\pgfqpoint{1.072cm}{1.728cm}}{\pgfqpoint{1.087cm}{1.694cm}}{\pgfqpoint{1.112cm}{1.668cm}}
\pgfpathcurveto{\pgfqpoint{1.138cm}{1.642cm}}{\pgfqpoint{1.172cm}{1.628cm}}{\pgfqpoint{1.209cm}{1.628cm}}
\pgfpathcurveto{\pgfqpoint{1.245cm}{1.628cm}}{\pgfqpoint{1.28cm}{1.642cm}}{\pgfqpoint{1.305cm}{1.668cm}}
\pgfpathcurveto{\pgfqpoint{1.331cm}{1.694cm}}{\pgfqpoint{1.345cm}{1.728cm}}{\pgfqpoint{1.345cm}{1.765cm}}
\pgfusepath{fill}
\begin{pgfscope}
\pgfsetdash{}{0cm}
\pgfsetlinewidth{0.818mm}
\pgfsetroundcap
\pgfsetroundjoin
\pgfsetmiterlimit{7.0}
\pgfpathmoveto{\pgfqpoint{0.682cm}{1.065cm}}
\pgfpathlineto{\pgfqpoint{1.246cm}{0.315cm}}
\pgfpathlineto{\pgfqpoint{1.811cm}{1.065cm}}
\pgfusepath{stroke}
\end{pgfscope}
\pgfpathmoveto{\pgfqpoint{1.948cm}{1.065cm}}
\pgfpathcurveto{\pgfqpoint{1.948cm}{1.101cm}}{\pgfqpoint{1.933cm}{1.136cm}}{\pgfqpoint{1.907cm}{1.162cm}}
\pgfpathcurveto{\pgfqpoint{1.882cm}{1.187cm}}{\pgfqpoint{1.847cm}{1.202cm}}{\pgfqpoint{1.811cm}{1.202cm}}
\pgfpathcurveto{\pgfqpoint{1.775cm}{1.202cm}}{\pgfqpoint{1.74cm}{1.187cm}}{\pgfqpoint{1.714cm}{1.162cm}}
\pgfpathcurveto{\pgfqpoint{1.689cm}{1.136cm}}{\pgfqpoint{1.674cm}{1.101cm}}{\pgfqpoint{1.674cm}{1.065cm}}
\pgfpathcurveto{\pgfqpoint{1.674cm}{1.029cm}}{\pgfqpoint{1.689cm}{0.994cm}}{\pgfqpoint{1.714cm}{0.968cm}}
\pgfpathcurveto{\pgfqpoint{1.74cm}{0.942cm}}{\pgfqpoint{1.775cm}{0.928cm}}{\pgfqpoint{1.811cm}{0.928cm}}
\pgfpathcurveto{\pgfqpoint{1.847cm}{0.928cm}}{\pgfqpoint{1.882cm}{0.942cm}}{\pgfqpoint{1.907cm}{0.968cm}}
\pgfpathcurveto{\pgfqpoint{1.933cm}{0.994cm}}{\pgfqpoint{1.948cm}{1.029cm}}{\pgfqpoint{1.948cm}{1.065cm}}
\pgfusepath{fill}
\begin{pgfscope}
\pgfsetdash{}{0cm}
\pgfsetlinewidth{0.818mm}
\pgfsetmiterlimit{4.0}
\pgfpathmoveto{\pgfqpoint{1.383cm}{0.178cm}}
\pgfpathcurveto{\pgfqpoint{1.383cm}{0.214cm}}{\pgfqpoint{1.369cm}{0.249cm}}{\pgfqpoint{1.343cm}{0.275cm}}
\pgfpathcurveto{\pgfqpoint{1.317cm}{0.3cm}}{\pgfqpoint{1.283cm}{0.315cm}}{\pgfqpoint{1.246cm}{0.315cm}}
\pgfpathcurveto{\pgfqpoint{1.21cm}{0.315cm}}{\pgfqpoint{1.175cm}{0.3cm}}{\pgfqpoint{1.15cm}{0.275cm}}
\pgfpathcurveto{\pgfqpoint{1.124cm}{0.249cm}}{\pgfqpoint{1.11cm}{0.214cm}}{\pgfqpoint{1.11cm}{0.178cm}}
\pgfpathcurveto{\pgfqpoint{1.11cm}{0.141cm}}{\pgfqpoint{1.124cm}{0.107cm}}{\pgfqpoint{1.15cm}{0.081cm}}
\pgfpathcurveto{\pgfqpoint{1.175cm}{0.055cm}}{\pgfqpoint{1.21cm}{0.041cm}}{\pgfqpoint{1.246cm}{0.041cm}}
\pgfpathcurveto{\pgfqpoint{1.283cm}{0.041cm}}{\pgfqpoint{1.317cm}{0.055cm}}{\pgfqpoint{1.343cm}{0.081cm}}
\pgfpathcurveto{\pgfqpoint{1.369cm}{0.107cm}}{\pgfqpoint{1.383cm}{0.141cm}}{\pgfqpoint{1.383cm}{0.178cm}}
\pgfusepath{stroke}
\end{pgfscope}
\end{pgfscope}
\end{pgfscope}
\end{pgfscope}
\end{tikzpicture}}}\|_{\CC^{-1+\alpha}(\rho)}
&\lesssim \|\phi+\psi\|_{L^\infty(\rho)}\|\UU_>X^{\!\resizebox{!}{.8em}{
\begin{tikzpicture}
\pgfpathmoveto{\pgfqpoint{0cm}{-0.035cm}}
\pgfpathlineto{\pgfqpoint{1.976cm}{-0.035cm}}
\pgfpathlineto{\pgfqpoint{1.976cm}{1.94cm}}
\pgfpathlineto{\pgfqpoint{0cm}{1.94cm}}
\pgfpathclose
\pgfusepath{clip}
\begin{pgfscope}
\begin{pgfscope}
\pgfpathmoveto{\pgfqpoint{0cm}{-0.035cm}}
\pgfpathlineto{\pgfqpoint{1.976cm}{-0.035cm}}
\pgfpathlineto{\pgfqpoint{1.976cm}{1.94cm}}
\pgfpathlineto{\pgfqpoint{0cm}{1.94cm}}
\pgfpathclose
\pgfusepath{clip}
\begin{pgfscope}
\begin{pgfscope}
\pgfsetdash{}{0cm}
\pgfsetlinewidth{0.818mm}
\pgfsetroundcap
\pgfsetroundjoin
\pgfsetmiterlimit{7.0}
\definecolor{eps2pgf_color}{gray}{0}\pgfsetstrokecolor{eps2pgf_color}\pgfsetfillcolor{eps2pgf_color}
\pgfpathmoveto{\pgfqpoint{0.117cm}{1.815cm}}
\pgfpathlineto{\pgfqpoint{0.682cm}{1.065cm}}
\pgfpathlineto{\pgfqpoint{1.246cm}{1.815cm}}
\pgfusepath{stroke}
\end{pgfscope}
\definecolor{eps2pgf_color}{gray}{0}\pgfsetstrokecolor{eps2pgf_color}\pgfsetfillcolor{eps2pgf_color}
\pgfpathmoveto{\pgfqpoint{0.273cm}{1.789cm}}
\pgfpathcurveto{\pgfqpoint{0.273cm}{1.825cm}}{\pgfqpoint{0.259cm}{1.86cm}}{\pgfqpoint{0.233cm}{1.886cm}}
\pgfpathcurveto{\pgfqpoint{0.207cm}{1.912cm}}{\pgfqpoint{0.173cm}{1.926cm}}{\pgfqpoint{0.137cm}{1.926cm}}
\pgfpathcurveto{\pgfqpoint{0.1cm}{1.926cm}}{\pgfqpoint{0.066cm}{1.912cm}}{\pgfqpoint{0.04cm}{1.886cm}}
\pgfpathcurveto{\pgfqpoint{0.014cm}{1.86cm}}{\pgfqpoint{0cm}{1.825cm}}{\pgfqpoint{0cm}{1.789cm}}
\pgfpathcurveto{\pgfqpoint{0cm}{1.753cm}}{\pgfqpoint{0.014cm}{1.718cm}}{\pgfqpoint{0.04cm}{1.692cm}}
\pgfpathcurveto{\pgfqpoint{0.066cm}{1.667cm}}{\pgfqpoint{0.1cm}{1.652cm}}{\pgfqpoint{0.137cm}{1.652cm}}
\pgfpathcurveto{\pgfqpoint{0.173cm}{1.652cm}}{\pgfqpoint{0.207cm}{1.667cm}}{\pgfqpoint{0.233cm}{1.692cm}}
\pgfpathcurveto{\pgfqpoint{0.259cm}{1.718cm}}{\pgfqpoint{0.273cm}{1.753cm}}{\pgfqpoint{0.273cm}{1.789cm}}
\pgfusepath{fill}
\begin{pgfscope}
\pgfsetdash{}{0cm}
\pgfsetlinewidth{0.818mm}
\pgfsetmiterlimit{7.0}
\pgfpathmoveto{\pgfqpoint{0.682cm}{1.065cm}}
\pgfpathlineto{\pgfqpoint{0.679cm}{1.812cm}}
\pgfusepath{stroke}
\end{pgfscope}
\pgfpathmoveto{\pgfqpoint{0.815cm}{1.793cm}}
\pgfpathcurveto{\pgfqpoint{0.815cm}{1.829cm}}{\pgfqpoint{0.801cm}{1.864cm}}{\pgfqpoint{0.775cm}{1.89cm}}
\pgfpathcurveto{\pgfqpoint{0.75cm}{1.915cm}}{\pgfqpoint{0.715cm}{1.93cm}}{\pgfqpoint{0.679cm}{1.93cm}}
\pgfpathcurveto{\pgfqpoint{0.643cm}{1.93cm}}{\pgfqpoint{0.608cm}{1.915cm}}{\pgfqpoint{0.582cm}{1.89cm}}
\pgfpathcurveto{\pgfqpoint{0.557cm}{1.864cm}}{\pgfqpoint{0.542cm}{1.829cm}}{\pgfqpoint{0.542cm}{1.793cm}}
\pgfpathcurveto{\pgfqpoint{0.542cm}{1.756cm}}{\pgfqpoint{0.557cm}{1.722cm}}{\pgfqpoint{0.582cm}{1.696cm}}
\pgfpathcurveto{\pgfqpoint{0.608cm}{1.67cm}}{\pgfqpoint{0.643cm}{1.656cm}}{\pgfqpoint{0.679cm}{1.656cm}}
\pgfpathcurveto{\pgfqpoint{0.715cm}{1.656cm}}{\pgfqpoint{0.75cm}{1.67cm}}{\pgfqpoint{0.775cm}{1.696cm}}
\pgfpathcurveto{\pgfqpoint{0.801cm}{1.722cm}}{\pgfqpoint{0.815cm}{1.756cm}}{\pgfqpoint{0.815cm}{1.793cm}}
\pgfusepath{fill}
\pgfpathmoveto{\pgfqpoint{1.345cm}{1.765cm}}
\pgfpathcurveto{\pgfqpoint{1.345cm}{1.801cm}}{\pgfqpoint{1.331cm}{1.836cm}}{\pgfqpoint{1.305cm}{1.862cm}}
\pgfpathcurveto{\pgfqpoint{1.28cm}{1.887cm}}{\pgfqpoint{1.245cm}{1.902cm}}{\pgfqpoint{1.209cm}{1.902cm}}
\pgfpathcurveto{\pgfqpoint{1.172cm}{1.902cm}}{\pgfqpoint{1.138cm}{1.887cm}}{\pgfqpoint{1.112cm}{1.862cm}}
\pgfpathcurveto{\pgfqpoint{1.087cm}{1.836cm}}{\pgfqpoint{1.072cm}{1.801cm}}{\pgfqpoint{1.072cm}{1.765cm}}
\pgfpathcurveto{\pgfqpoint{1.072cm}{1.728cm}}{\pgfqpoint{1.087cm}{1.694cm}}{\pgfqpoint{1.112cm}{1.668cm}}
\pgfpathcurveto{\pgfqpoint{1.138cm}{1.642cm}}{\pgfqpoint{1.172cm}{1.628cm}}{\pgfqpoint{1.209cm}{1.628cm}}
\pgfpathcurveto{\pgfqpoint{1.245cm}{1.628cm}}{\pgfqpoint{1.28cm}{1.642cm}}{\pgfqpoint{1.305cm}{1.668cm}}
\pgfpathcurveto{\pgfqpoint{1.331cm}{1.694cm}}{\pgfqpoint{1.345cm}{1.728cm}}{\pgfqpoint{1.345cm}{1.765cm}}
\pgfusepath{fill}
\begin{pgfscope}
\pgfsetdash{}{0cm}
\pgfsetlinewidth{0.818mm}
\pgfsetroundcap
\pgfsetroundjoin
\pgfsetmiterlimit{7.0}
\pgfpathmoveto{\pgfqpoint{0.682cm}{1.065cm}}
\pgfpathlineto{\pgfqpoint{1.246cm}{0.315cm}}
\pgfpathlineto{\pgfqpoint{1.811cm}{1.065cm}}
\pgfusepath{stroke}
\end{pgfscope}
\pgfpathmoveto{\pgfqpoint{1.948cm}{1.065cm}}
\pgfpathcurveto{\pgfqpoint{1.948cm}{1.101cm}}{\pgfqpoint{1.933cm}{1.136cm}}{\pgfqpoint{1.907cm}{1.162cm}}
\pgfpathcurveto{\pgfqpoint{1.882cm}{1.187cm}}{\pgfqpoint{1.847cm}{1.202cm}}{\pgfqpoint{1.811cm}{1.202cm}}
\pgfpathcurveto{\pgfqpoint{1.775cm}{1.202cm}}{\pgfqpoint{1.74cm}{1.187cm}}{\pgfqpoint{1.714cm}{1.162cm}}
\pgfpathcurveto{\pgfqpoint{1.689cm}{1.136cm}}{\pgfqpoint{1.674cm}{1.101cm}}{\pgfqpoint{1.674cm}{1.065cm}}
\pgfpathcurveto{\pgfqpoint{1.674cm}{1.029cm}}{\pgfqpoint{1.689cm}{0.994cm}}{\pgfqpoint{1.714cm}{0.968cm}}
\pgfpathcurveto{\pgfqpoint{1.74cm}{0.942cm}}{\pgfqpoint{1.775cm}{0.928cm}}{\pgfqpoint{1.811cm}{0.928cm}}
\pgfpathcurveto{\pgfqpoint{1.847cm}{0.928cm}}{\pgfqpoint{1.882cm}{0.942cm}}{\pgfqpoint{1.907cm}{0.968cm}}
\pgfpathcurveto{\pgfqpoint{1.933cm}{0.994cm}}{\pgfqpoint{1.948cm}{1.029cm}}{\pgfqpoint{1.948cm}{1.065cm}}
\pgfusepath{fill}
\begin{pgfscope}
\pgfsetdash{}{0cm}
\pgfsetlinewidth{0.818mm}
\pgfsetmiterlimit{4.0}
\pgfpathmoveto{\pgfqpoint{1.383cm}{0.178cm}}
\pgfpathcurveto{\pgfqpoint{1.383cm}{0.214cm}}{\pgfqpoint{1.369cm}{0.249cm}}{\pgfqpoint{1.343cm}{0.275cm}}
\pgfpathcurveto{\pgfqpoint{1.317cm}{0.3cm}}{\pgfqpoint{1.283cm}{0.315cm}}{\pgfqpoint{1.246cm}{0.315cm}}
\pgfpathcurveto{\pgfqpoint{1.21cm}{0.315cm}}{\pgfqpoint{1.175cm}{0.3cm}}{\pgfqpoint{1.15cm}{0.275cm}}
\pgfpathcurveto{\pgfqpoint{1.124cm}{0.249cm}}{\pgfqpoint{1.11cm}{0.214cm}}{\pgfqpoint{1.11cm}{0.178cm}}
\pgfpathcurveto{\pgfqpoint{1.11cm}{0.141cm}}{\pgfqpoint{1.124cm}{0.107cm}}{\pgfqpoint{1.15cm}{0.081cm}}
\pgfpathcurveto{\pgfqpoint{1.175cm}{0.055cm}}{\pgfqpoint{1.21cm}{0.041cm}}{\pgfqpoint{1.246cm}{0.041cm}}
\pgfpathcurveto{\pgfqpoint{1.283cm}{0.041cm}}{\pgfqpoint{1.317cm}{0.055cm}}{\pgfqpoint{1.343cm}{0.081cm}}
\pgfpathcurveto{\pgfqpoint{1.369cm}{0.107cm}}{\pgfqpoint{1.383cm}{0.141cm}}{\pgfqpoint{1.383cm}{0.178cm}}
\pgfusepath{stroke}
\end{pgfscope}
\end{pgfscope}
\end{pgfscope}
\end{pgfscope}
\end{tikzpicture}}}\|_{\CC^{-1+\alpha}}\\
&\lesssim 2^{-(1-\alpha-\kappa)K/2}\|\phi+\psi\|_{L^\infty(\rho)}\\
&\lesssim 1+\|\psi\|_{L^\infty(\rho)}^\varepsilon,
\end{align*}
and finally
\begin{align*}
\|3\UU_> X\succ(\phi+\psi)^2\|_{\CC^{-1+\alpha}(\rho^{2+\alpha})}
&\lesssim \|\phi+\psi\|^2_{L^\infty(\rho)}\|\UU_>X\|_{\CC^{-1+\alpha}(\rho^\alpha)}\\
&\lesssim 2^{-(\frac{1}{2}-\alpha-\kappa)\frac{4}{3}K}\|\phi+\psi\|_{L^\infty(\rho)}^2\\
&\lesssim 1+\|\psi\|_{L^\infty(\rho)}^{1+\varepsilon}.
\end{align*}
Hence, we have shown that
\begin{align*}
\|\Phi +3\llbracket X^2 \rrbracket\succX^{\!\resizebox{0.6em}{!}{
\begin{tikzpicture}
\pgfpathmoveto{\pgfqpoint{0cm}{-0.035cm}}
\pgfpathlineto{\pgfqpoint{1.376cm}{-0.035cm}}
\pgfpathlineto{\pgfqpoint{1.376cm}{1.552cm}}
\pgfpathlineto{\pgfqpoint{0cm}{1.552cm}}
\pgfpathclose
\pgfusepath{clip}
\begin{pgfscope}
\begin{pgfscope}
\pgfpathmoveto{\pgfqpoint{0cm}{-0.035cm}}
\pgfpathlineto{\pgfqpoint{1.376cm}{-0.035cm}}
\pgfpathlineto{\pgfqpoint{1.376cm}{1.552cm}}
\pgfpathlineto{\pgfqpoint{0cm}{1.552cm}}
\pgfpathclose
\pgfusepath{clip}
\begin{pgfscope}
\begin{pgfscope}
\pgfsetdash{}{0cm}
\pgfsetlinewidth{0.818mm}
\pgfsetroundcap
\pgfsetroundjoin
\pgfsetmiterlimit{7.0}
\definecolor{eps2pgf_color}{gray}{0}\pgfsetstrokecolor{eps2pgf_color}\pgfsetfillcolor{eps2pgf_color}
\pgfpathmoveto{\pgfqpoint{0.117cm}{1.421cm}}
\pgfpathlineto{\pgfqpoint{0.682cm}{0.671cm}}
\pgfpathlineto{\pgfqpoint{1.246cm}{1.421cm}}
\pgfusepath{stroke}
\end{pgfscope}
\definecolor{eps2pgf_color}{gray}{0}\pgfsetstrokecolor{eps2pgf_color}\pgfsetfillcolor{eps2pgf_color}
\pgfpathmoveto{\pgfqpoint{0.273cm}{1.395cm}}
\pgfpathcurveto{\pgfqpoint{0.273cm}{1.432cm}}{\pgfqpoint{0.259cm}{1.467cm}}{\pgfqpoint{0.233cm}{1.492cm}}
\pgfpathcurveto{\pgfqpoint{0.207cm}{1.518cm}}{\pgfqpoint{0.173cm}{1.532cm}}{\pgfqpoint{0.137cm}{1.532cm}}
\pgfpathcurveto{\pgfqpoint{0.1cm}{1.532cm}}{\pgfqpoint{0.066cm}{1.518cm}}{\pgfqpoint{0.04cm}{1.492cm}}
\pgfpathcurveto{\pgfqpoint{0.014cm}{1.467cm}}{\pgfqpoint{0cm}{1.432cm}}{\pgfqpoint{0cm}{1.395cm}}
\pgfpathcurveto{\pgfqpoint{0cm}{1.359cm}}{\pgfqpoint{0.014cm}{1.324cm}}{\pgfqpoint{0.04cm}{1.299cm}}
\pgfpathcurveto{\pgfqpoint{0.066cm}{1.273cm}}{\pgfqpoint{0.1cm}{1.258cm}}{\pgfqpoint{0.137cm}{1.258cm}}
\pgfpathcurveto{\pgfqpoint{0.173cm}{1.258cm}}{\pgfqpoint{0.207cm}{1.273cm}}{\pgfqpoint{0.233cm}{1.299cm}}
\pgfpathcurveto{\pgfqpoint{0.259cm}{1.324cm}}{\pgfqpoint{0.273cm}{1.359cm}}{\pgfqpoint{0.273cm}{1.395cm}}
\pgfusepath{fill}
\begin{pgfscope}
\pgfsetdash{}{0cm}
\pgfsetlinewidth{0.818mm}
\pgfsetmiterlimit{7.0}
\pgfpathmoveto{\pgfqpoint{0.682cm}{0.671cm}}
\pgfpathlineto{\pgfqpoint{0.679cm}{1.418cm}}
\pgfusepath{stroke}
\end{pgfscope}
\pgfpathmoveto{\pgfqpoint{0.815cm}{1.399cm}}
\pgfpathcurveto{\pgfqpoint{0.815cm}{1.435cm}}{\pgfqpoint{0.801cm}{1.47cm}}{\pgfqpoint{0.775cm}{1.496cm}}
\pgfpathcurveto{\pgfqpoint{0.75cm}{1.521cm}}{\pgfqpoint{0.715cm}{1.536cm}}{\pgfqpoint{0.679cm}{1.536cm}}
\pgfpathcurveto{\pgfqpoint{0.643cm}{1.536cm}}{\pgfqpoint{0.608cm}{1.521cm}}{\pgfqpoint{0.582cm}{1.496cm}}
\pgfpathcurveto{\pgfqpoint{0.557cm}{1.47cm}}{\pgfqpoint{0.542cm}{1.435cm}}{\pgfqpoint{0.542cm}{1.399cm}}
\pgfpathcurveto{\pgfqpoint{0.542cm}{1.363cm}}{\pgfqpoint{0.557cm}{1.328cm}}{\pgfqpoint{0.582cm}{1.302cm}}
\pgfpathcurveto{\pgfqpoint{0.608cm}{1.276cm}}{\pgfqpoint{0.643cm}{1.262cm}}{\pgfqpoint{0.679cm}{1.262cm}}
\pgfpathcurveto{\pgfqpoint{0.715cm}{1.262cm}}{\pgfqpoint{0.75cm}{1.276cm}}{\pgfqpoint{0.775cm}{1.302cm}}
\pgfpathcurveto{\pgfqpoint{0.801cm}{1.328cm}}{\pgfqpoint{0.815cm}{1.363cm}}{\pgfqpoint{0.815cm}{1.399cm}}
\pgfusepath{fill}
\pgfpathmoveto{\pgfqpoint{1.345cm}{1.371cm}}
\pgfpathcurveto{\pgfqpoint{1.345cm}{1.408cm}}{\pgfqpoint{1.331cm}{1.442cm}}{\pgfqpoint{1.305cm}{1.468cm}}
\pgfpathcurveto{\pgfqpoint{1.28cm}{1.494cm}}{\pgfqpoint{1.245cm}{1.508cm}}{\pgfqpoint{1.209cm}{1.508cm}}
\pgfpathcurveto{\pgfqpoint{1.172cm}{1.508cm}}{\pgfqpoint{1.138cm}{1.494cm}}{\pgfqpoint{1.112cm}{1.468cm}}
\pgfpathcurveto{\pgfqpoint{1.087cm}{1.442cm}}{\pgfqpoint{1.072cm}{1.408cm}}{\pgfqpoint{1.072cm}{1.371cm}}
\pgfpathcurveto{\pgfqpoint{1.072cm}{1.335cm}}{\pgfqpoint{1.087cm}{1.3cm}}{\pgfqpoint{1.112cm}{1.274cm}}
\pgfpathcurveto{\pgfqpoint{1.138cm}{1.249cm}}{\pgfqpoint{1.172cm}{1.234cm}}{\pgfqpoint{1.209cm}{1.234cm}}
\pgfpathcurveto{\pgfqpoint{1.245cm}{1.234cm}}{\pgfqpoint{1.28cm}{1.249cm}}{\pgfqpoint{1.305cm}{1.274cm}}
\pgfpathcurveto{\pgfqpoint{1.331cm}{1.3cm}}{\pgfqpoint{1.345cm}{1.335cm}}{\pgfqpoint{1.345cm}{1.371cm}}
\pgfusepath{fill}
\begin{pgfscope}
\pgfsetdash{}{0cm}
\pgfsetlinewidth{0.818mm}
\pgfsetroundcap
\pgfsetmiterlimit{4.0}
\pgfpathmoveto{\pgfqpoint{0.682cm}{0.671cm}}
\pgfpathlineto{\pgfqpoint{0.682cm}{0.042cm}}
\pgfusepath{stroke}
\end{pgfscope}
\end{pgfscope}
\end{pgfscope}
\end{pgfscope}
\end{tikzpicture}}}-3\UU_>\llbracket X^2 \rrbracket\succ(\phi+\psi)\|_{\CC^{-1+\alpha}(\rho^{2+\alpha})}\lesssim 1+\|\psi\|_{L^\infty(\rho)}^{1+\varepsilon}+ \|\psi\|_{\CC^{\frac{1}{2}+\alpha}(\rho^{\frac{3}{2}+\alpha})} .
\end{align*}
Similarly,
\begin{align*}
\|3\UU_\leq\llbracket X^2 \rrbracket\succ(\phi+\psi)\|_{\CC^{-1+\alpha}(\rho^{2+\alpha})}&\lesssim \|\phi+\psi\|_{L^\infty(\rho)}\|\UU_\leq\llbracket X^2 \rrbracket\|_{\CC^{-1+\alpha}(\rho^{1+\alpha})}\\
&\lesssim 2^{(\alpha+\kappa)K}(1+\|\psi\|_{L^\infty(\rho)} )\\
&\lesssim 1+\|\psi\|_{L^\infty(\rho)}^{1+\varepsilon},
\end{align*}
and for the commutator, we obtain
\begin{align*}
&\|3[\Q,(-X^{\!\resizebox{0.6em}{!}{
\begin{tikzpicture}
\pgfpathmoveto{\pgfqpoint{0cm}{-0.035cm}}
\pgfpathlineto{\pgfqpoint{1.376cm}{-0.035cm}}
\pgfpathlineto{\pgfqpoint{1.376cm}{1.552cm}}
\pgfpathlineto{\pgfqpoint{0cm}{1.552cm}}
\pgfpathclose
\pgfusepath{clip}
\begin{pgfscope}
\begin{pgfscope}
\pgfpathmoveto{\pgfqpoint{0cm}{-0.035cm}}
\pgfpathlineto{\pgfqpoint{1.376cm}{-0.035cm}}
\pgfpathlineto{\pgfqpoint{1.376cm}{1.552cm}}
\pgfpathlineto{\pgfqpoint{0cm}{1.552cm}}
\pgfpathclose
\pgfusepath{clip}
\begin{pgfscope}
\begin{pgfscope}
\pgfsetdash{}{0cm}
\pgfsetlinewidth{0.818mm}
\pgfsetroundcap
\pgfsetroundjoin
\pgfsetmiterlimit{7.0}
\definecolor{eps2pgf_color}{gray}{0}\pgfsetstrokecolor{eps2pgf_color}\pgfsetfillcolor{eps2pgf_color}
\pgfpathmoveto{\pgfqpoint{0.117cm}{1.421cm}}
\pgfpathlineto{\pgfqpoint{0.682cm}{0.671cm}}
\pgfpathlineto{\pgfqpoint{1.246cm}{1.421cm}}
\pgfusepath{stroke}
\end{pgfscope}
\definecolor{eps2pgf_color}{gray}{0}\pgfsetstrokecolor{eps2pgf_color}\pgfsetfillcolor{eps2pgf_color}
\pgfpathmoveto{\pgfqpoint{0.273cm}{1.395cm}}
\pgfpathcurveto{\pgfqpoint{0.273cm}{1.432cm}}{\pgfqpoint{0.259cm}{1.467cm}}{\pgfqpoint{0.233cm}{1.492cm}}
\pgfpathcurveto{\pgfqpoint{0.207cm}{1.518cm}}{\pgfqpoint{0.173cm}{1.532cm}}{\pgfqpoint{0.137cm}{1.532cm}}
\pgfpathcurveto{\pgfqpoint{0.1cm}{1.532cm}}{\pgfqpoint{0.066cm}{1.518cm}}{\pgfqpoint{0.04cm}{1.492cm}}
\pgfpathcurveto{\pgfqpoint{0.014cm}{1.467cm}}{\pgfqpoint{0cm}{1.432cm}}{\pgfqpoint{0cm}{1.395cm}}
\pgfpathcurveto{\pgfqpoint{0cm}{1.359cm}}{\pgfqpoint{0.014cm}{1.324cm}}{\pgfqpoint{0.04cm}{1.299cm}}
\pgfpathcurveto{\pgfqpoint{0.066cm}{1.273cm}}{\pgfqpoint{0.1cm}{1.258cm}}{\pgfqpoint{0.137cm}{1.258cm}}
\pgfpathcurveto{\pgfqpoint{0.173cm}{1.258cm}}{\pgfqpoint{0.207cm}{1.273cm}}{\pgfqpoint{0.233cm}{1.299cm}}
\pgfpathcurveto{\pgfqpoint{0.259cm}{1.324cm}}{\pgfqpoint{0.273cm}{1.359cm}}{\pgfqpoint{0.273cm}{1.395cm}}
\pgfusepath{fill}
\begin{pgfscope}
\pgfsetdash{}{0cm}
\pgfsetlinewidth{0.818mm}
\pgfsetmiterlimit{7.0}
\pgfpathmoveto{\pgfqpoint{0.682cm}{0.671cm}}
\pgfpathlineto{\pgfqpoint{0.679cm}{1.418cm}}
\pgfusepath{stroke}
\end{pgfscope}
\pgfpathmoveto{\pgfqpoint{0.815cm}{1.399cm}}
\pgfpathcurveto{\pgfqpoint{0.815cm}{1.435cm}}{\pgfqpoint{0.801cm}{1.47cm}}{\pgfqpoint{0.775cm}{1.496cm}}
\pgfpathcurveto{\pgfqpoint{0.75cm}{1.521cm}}{\pgfqpoint{0.715cm}{1.536cm}}{\pgfqpoint{0.679cm}{1.536cm}}
\pgfpathcurveto{\pgfqpoint{0.643cm}{1.536cm}}{\pgfqpoint{0.608cm}{1.521cm}}{\pgfqpoint{0.582cm}{1.496cm}}
\pgfpathcurveto{\pgfqpoint{0.557cm}{1.47cm}}{\pgfqpoint{0.542cm}{1.435cm}}{\pgfqpoint{0.542cm}{1.399cm}}
\pgfpathcurveto{\pgfqpoint{0.542cm}{1.363cm}}{\pgfqpoint{0.557cm}{1.328cm}}{\pgfqpoint{0.582cm}{1.302cm}}
\pgfpathcurveto{\pgfqpoint{0.608cm}{1.276cm}}{\pgfqpoint{0.643cm}{1.262cm}}{\pgfqpoint{0.679cm}{1.262cm}}
\pgfpathcurveto{\pgfqpoint{0.715cm}{1.262cm}}{\pgfqpoint{0.75cm}{1.276cm}}{\pgfqpoint{0.775cm}{1.302cm}}
\pgfpathcurveto{\pgfqpoint{0.801cm}{1.328cm}}{\pgfqpoint{0.815cm}{1.363cm}}{\pgfqpoint{0.815cm}{1.399cm}}
\pgfusepath{fill}
\pgfpathmoveto{\pgfqpoint{1.345cm}{1.371cm}}
\pgfpathcurveto{\pgfqpoint{1.345cm}{1.408cm}}{\pgfqpoint{1.331cm}{1.442cm}}{\pgfqpoint{1.305cm}{1.468cm}}
\pgfpathcurveto{\pgfqpoint{1.28cm}{1.494cm}}{\pgfqpoint{1.245cm}{1.508cm}}{\pgfqpoint{1.209cm}{1.508cm}}
\pgfpathcurveto{\pgfqpoint{1.172cm}{1.508cm}}{\pgfqpoint{1.138cm}{1.494cm}}{\pgfqpoint{1.112cm}{1.468cm}}
\pgfpathcurveto{\pgfqpoint{1.087cm}{1.442cm}}{\pgfqpoint{1.072cm}{1.408cm}}{\pgfqpoint{1.072cm}{1.371cm}}
\pgfpathcurveto{\pgfqpoint{1.072cm}{1.335cm}}{\pgfqpoint{1.087cm}{1.3cm}}{\pgfqpoint{1.112cm}{1.274cm}}
\pgfpathcurveto{\pgfqpoint{1.138cm}{1.249cm}}{\pgfqpoint{1.172cm}{1.234cm}}{\pgfqpoint{1.209cm}{1.234cm}}
\pgfpathcurveto{\pgfqpoint{1.245cm}{1.234cm}}{\pgfqpoint{1.28cm}{1.249cm}}{\pgfqpoint{1.305cm}{1.274cm}}
\pgfpathcurveto{\pgfqpoint{1.331cm}{1.3cm}}{\pgfqpoint{1.345cm}{1.335cm}}{\pgfqpoint{1.345cm}{1.371cm}}
\pgfusepath{fill}
\begin{pgfscope}
\pgfsetdash{}{0cm}
\pgfsetlinewidth{0.818mm}
\pgfsetroundcap
\pgfsetmiterlimit{4.0}
\pgfpathmoveto{\pgfqpoint{0.682cm}{0.671cm}}
\pgfpathlineto{\pgfqpoint{0.682cm}{0.042cm}}
\pgfusepath{stroke}
\end{pgfscope}
\end{pgfscope}
\end{pgfscope}
\end{pgfscope}
\end{tikzpicture}}}+\phi+\psi)\prec]X^{\!\resizebox{0.6em}{!}{
\begin{tikzpicture}
\pgfpathmoveto{\pgfqpoint{0cm}{0cm}}
\pgfpathlineto{\pgfqpoint{1.376cm}{0cm}}
\pgfpathlineto{\pgfqpoint{1.376cm}{1.588cm}}
\pgfpathlineto{\pgfqpoint{0cm}{1.588cm}}
\pgfpathclose
\pgfusepath{clip}
\begin{pgfscope}
\begin{pgfscope}
\pgfpathmoveto{\pgfqpoint{0cm}{0cm}}
\pgfpathlineto{\pgfqpoint{1.376cm}{0cm}}
\pgfpathlineto{\pgfqpoint{1.376cm}{1.588cm}}
\pgfpathlineto{\pgfqpoint{0cm}{1.588cm}}
\pgfpathclose
\pgfusepath{clip}
\begin{pgfscope}
\begin{pgfscope}
\definecolor{eps2pgf_color}{gray}{0.976471}\pgfsetstrokecolor{eps2pgf_color}\pgfsetfillcolor{eps2pgf_color}
\pgfpathmoveto{\pgfqpoint{0cm}{0cm}}
\pgfpathlineto{\pgfqpoint{1.376cm}{0cm}}
\pgfpathlineto{\pgfqpoint{1.376cm}{1.588cm}}
\pgfpathlineto{\pgfqpoint{0cm}{1.588cm}}
\pgfpathclose
\pgfusepath{fill}
\end{pgfscope}
\begin{pgfscope}
\pgfsetdash{}{0cm}
\pgfsetlinewidth{0.818mm}
\pgfsetroundcap
\pgfsetroundjoin
\pgfsetmiterlimit{7.0}
\definecolor{eps2pgf_color}{gray}{0}\pgfsetstrokecolor{eps2pgf_color}\pgfsetfillcolor{eps2pgf_color}
\pgfpathmoveto{\pgfqpoint{0.117cm}{1.476cm}}
\pgfpathlineto{\pgfqpoint{0.682cm}{0.726cm}}
\pgfpathlineto{\pgfqpoint{1.246cm}{1.476cm}}
\pgfusepath{stroke}
\end{pgfscope}
\definecolor{eps2pgf_color}{gray}{0}\pgfsetstrokecolor{eps2pgf_color}\pgfsetfillcolor{eps2pgf_color}
\pgfpathmoveto{\pgfqpoint{0.273cm}{1.451cm}}
\pgfpathcurveto{\pgfqpoint{0.273cm}{1.487cm}}{\pgfqpoint{0.259cm}{1.522cm}}{\pgfqpoint{0.233cm}{1.547cm}}
\pgfpathcurveto{\pgfqpoint{0.207cm}{1.573cm}}{\pgfqpoint{0.173cm}{1.588cm}}{\pgfqpoint{0.137cm}{1.588cm}}
\pgfpathcurveto{\pgfqpoint{0.1cm}{1.588cm}}{\pgfqpoint{0.066cm}{1.573cm}}{\pgfqpoint{0.04cm}{1.547cm}}
\pgfpathcurveto{\pgfqpoint{0.014cm}{1.522cm}}{\pgfqpoint{0cm}{1.487cm}}{\pgfqpoint{0cm}{1.451cm}}
\pgfpathcurveto{\pgfqpoint{0cm}{1.414cm}}{\pgfqpoint{0.014cm}{1.379cm}}{\pgfqpoint{0.04cm}{1.354cm}}
\pgfpathcurveto{\pgfqpoint{0.066cm}{1.328cm}}{\pgfqpoint{0.1cm}{1.314cm}}{\pgfqpoint{0.137cm}{1.314cm}}
\pgfpathcurveto{\pgfqpoint{0.173cm}{1.314cm}}{\pgfqpoint{0.207cm}{1.328cm}}{\pgfqpoint{0.233cm}{1.354cm}}
\pgfpathcurveto{\pgfqpoint{0.259cm}{1.379cm}}{\pgfqpoint{0.273cm}{1.414cm}}{\pgfqpoint{0.273cm}{1.451cm}}
\pgfusepath{fill}
\pgfpathmoveto{\pgfqpoint{1.345cm}{1.426cm}}
\pgfpathcurveto{\pgfqpoint{1.345cm}{1.463cm}}{\pgfqpoint{1.331cm}{1.497cm}}{\pgfqpoint{1.305cm}{1.523cm}}
\pgfpathcurveto{\pgfqpoint{1.28cm}{1.549cm}}{\pgfqpoint{1.245cm}{1.563cm}}{\pgfqpoint{1.209cm}{1.563cm}}
\pgfpathcurveto{\pgfqpoint{1.172cm}{1.563cm}}{\pgfqpoint{1.138cm}{1.549cm}}{\pgfqpoint{1.112cm}{1.523cm}}
\pgfpathcurveto{\pgfqpoint{1.087cm}{1.497cm}}{\pgfqpoint{1.072cm}{1.463cm}}{\pgfqpoint{1.072cm}{1.426cm}}
\pgfpathcurveto{\pgfqpoint{1.072cm}{1.39cm}}{\pgfqpoint{1.087cm}{1.355cm}}{\pgfqpoint{1.112cm}{1.329cm}}
\pgfpathcurveto{\pgfqpoint{1.138cm}{1.304cm}}{\pgfqpoint{1.172cm}{1.289cm}}{\pgfqpoint{1.209cm}{1.289cm}}
\pgfpathcurveto{\pgfqpoint{1.245cm}{1.289cm}}{\pgfqpoint{1.28cm}{1.304cm}}{\pgfqpoint{1.305cm}{1.329cm}}
\pgfpathcurveto{\pgfqpoint{1.331cm}{1.355cm}}{\pgfqpoint{1.345cm}{1.39cm}}{\pgfqpoint{1.345cm}{1.426cm}}
\pgfusepath{fill}
\begin{pgfscope}
\pgfsetdash{}{0cm}
\pgfsetlinewidth{0.818mm}
\pgfsetroundcap
\pgfsetmiterlimit{4.0}
\pgfpathmoveto{\pgfqpoint{0.682cm}{0.726cm}}
\pgfpathlineto{\pgfqpoint{0.682cm}{0.097cm}}
\pgfusepath{stroke}
\end{pgfscope}
\end{pgfscope}
\end{pgfscope}
\end{pgfscope}
\end{tikzpicture}}}||_{\CC^{-1+\alpha}(\rho^{2+\alpha})}\\
&\quad=\|3\Delta(-X^{\!\resizebox{0.6em}{!}{
\begin{tikzpicture}
\pgfpathmoveto{\pgfqpoint{0cm}{-0.035cm}}
\pgfpathlineto{\pgfqpoint{1.376cm}{-0.035cm}}
\pgfpathlineto{\pgfqpoint{1.376cm}{1.552cm}}
\pgfpathlineto{\pgfqpoint{0cm}{1.552cm}}
\pgfpathclose
\pgfusepath{clip}
\begin{pgfscope}
\begin{pgfscope}
\pgfpathmoveto{\pgfqpoint{0cm}{-0.035cm}}
\pgfpathlineto{\pgfqpoint{1.376cm}{-0.035cm}}
\pgfpathlineto{\pgfqpoint{1.376cm}{1.552cm}}
\pgfpathlineto{\pgfqpoint{0cm}{1.552cm}}
\pgfpathclose
\pgfusepath{clip}
\begin{pgfscope}
\begin{pgfscope}
\pgfsetdash{}{0cm}
\pgfsetlinewidth{0.818mm}
\pgfsetroundcap
\pgfsetroundjoin
\pgfsetmiterlimit{7.0}
\definecolor{eps2pgf_color}{gray}{0}\pgfsetstrokecolor{eps2pgf_color}\pgfsetfillcolor{eps2pgf_color}
\pgfpathmoveto{\pgfqpoint{0.117cm}{1.421cm}}
\pgfpathlineto{\pgfqpoint{0.682cm}{0.671cm}}
\pgfpathlineto{\pgfqpoint{1.246cm}{1.421cm}}
\pgfusepath{stroke}
\end{pgfscope}
\definecolor{eps2pgf_color}{gray}{0}\pgfsetstrokecolor{eps2pgf_color}\pgfsetfillcolor{eps2pgf_color}
\pgfpathmoveto{\pgfqpoint{0.273cm}{1.395cm}}
\pgfpathcurveto{\pgfqpoint{0.273cm}{1.432cm}}{\pgfqpoint{0.259cm}{1.467cm}}{\pgfqpoint{0.233cm}{1.492cm}}
\pgfpathcurveto{\pgfqpoint{0.207cm}{1.518cm}}{\pgfqpoint{0.173cm}{1.532cm}}{\pgfqpoint{0.137cm}{1.532cm}}
\pgfpathcurveto{\pgfqpoint{0.1cm}{1.532cm}}{\pgfqpoint{0.066cm}{1.518cm}}{\pgfqpoint{0.04cm}{1.492cm}}
\pgfpathcurveto{\pgfqpoint{0.014cm}{1.467cm}}{\pgfqpoint{0cm}{1.432cm}}{\pgfqpoint{0cm}{1.395cm}}
\pgfpathcurveto{\pgfqpoint{0cm}{1.359cm}}{\pgfqpoint{0.014cm}{1.324cm}}{\pgfqpoint{0.04cm}{1.299cm}}
\pgfpathcurveto{\pgfqpoint{0.066cm}{1.273cm}}{\pgfqpoint{0.1cm}{1.258cm}}{\pgfqpoint{0.137cm}{1.258cm}}
\pgfpathcurveto{\pgfqpoint{0.173cm}{1.258cm}}{\pgfqpoint{0.207cm}{1.273cm}}{\pgfqpoint{0.233cm}{1.299cm}}
\pgfpathcurveto{\pgfqpoint{0.259cm}{1.324cm}}{\pgfqpoint{0.273cm}{1.359cm}}{\pgfqpoint{0.273cm}{1.395cm}}
\pgfusepath{fill}
\begin{pgfscope}
\pgfsetdash{}{0cm}
\pgfsetlinewidth{0.818mm}
\pgfsetmiterlimit{7.0}
\pgfpathmoveto{\pgfqpoint{0.682cm}{0.671cm}}
\pgfpathlineto{\pgfqpoint{0.679cm}{1.418cm}}
\pgfusepath{stroke}
\end{pgfscope}
\pgfpathmoveto{\pgfqpoint{0.815cm}{1.399cm}}
\pgfpathcurveto{\pgfqpoint{0.815cm}{1.435cm}}{\pgfqpoint{0.801cm}{1.47cm}}{\pgfqpoint{0.775cm}{1.496cm}}
\pgfpathcurveto{\pgfqpoint{0.75cm}{1.521cm}}{\pgfqpoint{0.715cm}{1.536cm}}{\pgfqpoint{0.679cm}{1.536cm}}
\pgfpathcurveto{\pgfqpoint{0.643cm}{1.536cm}}{\pgfqpoint{0.608cm}{1.521cm}}{\pgfqpoint{0.582cm}{1.496cm}}
\pgfpathcurveto{\pgfqpoint{0.557cm}{1.47cm}}{\pgfqpoint{0.542cm}{1.435cm}}{\pgfqpoint{0.542cm}{1.399cm}}
\pgfpathcurveto{\pgfqpoint{0.542cm}{1.363cm}}{\pgfqpoint{0.557cm}{1.328cm}}{\pgfqpoint{0.582cm}{1.302cm}}
\pgfpathcurveto{\pgfqpoint{0.608cm}{1.276cm}}{\pgfqpoint{0.643cm}{1.262cm}}{\pgfqpoint{0.679cm}{1.262cm}}
\pgfpathcurveto{\pgfqpoint{0.715cm}{1.262cm}}{\pgfqpoint{0.75cm}{1.276cm}}{\pgfqpoint{0.775cm}{1.302cm}}
\pgfpathcurveto{\pgfqpoint{0.801cm}{1.328cm}}{\pgfqpoint{0.815cm}{1.363cm}}{\pgfqpoint{0.815cm}{1.399cm}}
\pgfusepath{fill}
\pgfpathmoveto{\pgfqpoint{1.345cm}{1.371cm}}
\pgfpathcurveto{\pgfqpoint{1.345cm}{1.408cm}}{\pgfqpoint{1.331cm}{1.442cm}}{\pgfqpoint{1.305cm}{1.468cm}}
\pgfpathcurveto{\pgfqpoint{1.28cm}{1.494cm}}{\pgfqpoint{1.245cm}{1.508cm}}{\pgfqpoint{1.209cm}{1.508cm}}
\pgfpathcurveto{\pgfqpoint{1.172cm}{1.508cm}}{\pgfqpoint{1.138cm}{1.494cm}}{\pgfqpoint{1.112cm}{1.468cm}}
\pgfpathcurveto{\pgfqpoint{1.087cm}{1.442cm}}{\pgfqpoint{1.072cm}{1.408cm}}{\pgfqpoint{1.072cm}{1.371cm}}
\pgfpathcurveto{\pgfqpoint{1.072cm}{1.335cm}}{\pgfqpoint{1.087cm}{1.3cm}}{\pgfqpoint{1.112cm}{1.274cm}}
\pgfpathcurveto{\pgfqpoint{1.138cm}{1.249cm}}{\pgfqpoint{1.172cm}{1.234cm}}{\pgfqpoint{1.209cm}{1.234cm}}
\pgfpathcurveto{\pgfqpoint{1.245cm}{1.234cm}}{\pgfqpoint{1.28cm}{1.249cm}}{\pgfqpoint{1.305cm}{1.274cm}}
\pgfpathcurveto{\pgfqpoint{1.331cm}{1.3cm}}{\pgfqpoint{1.345cm}{1.335cm}}{\pgfqpoint{1.345cm}{1.371cm}}
\pgfusepath{fill}
\begin{pgfscope}
\pgfsetdash{}{0cm}
\pgfsetlinewidth{0.818mm}
\pgfsetroundcap
\pgfsetmiterlimit{4.0}
\pgfpathmoveto{\pgfqpoint{0.682cm}{0.671cm}}
\pgfpathlineto{\pgfqpoint{0.682cm}{0.042cm}}
\pgfusepath{stroke}
\end{pgfscope}
\end{pgfscope}
\end{pgfscope}
\end{pgfscope}
\end{tikzpicture}}}+\phi+\psi)\prec X^{\!\resizebox{0.6em}{!}{
\begin{tikzpicture}
\pgfpathmoveto{\pgfqpoint{0cm}{0cm}}
\pgfpathlineto{\pgfqpoint{1.376cm}{0cm}}
\pgfpathlineto{\pgfqpoint{1.376cm}{1.588cm}}
\pgfpathlineto{\pgfqpoint{0cm}{1.588cm}}
\pgfpathclose
\pgfusepath{clip}
\begin{pgfscope}
\begin{pgfscope}
\pgfpathmoveto{\pgfqpoint{0cm}{0cm}}
\pgfpathlineto{\pgfqpoint{1.376cm}{0cm}}
\pgfpathlineto{\pgfqpoint{1.376cm}{1.588cm}}
\pgfpathlineto{\pgfqpoint{0cm}{1.588cm}}
\pgfpathclose
\pgfusepath{clip}
\begin{pgfscope}
\begin{pgfscope}
\definecolor{eps2pgf_color}{gray}{0.976471}\pgfsetstrokecolor{eps2pgf_color}\pgfsetfillcolor{eps2pgf_color}
\pgfpathmoveto{\pgfqpoint{0cm}{0cm}}
\pgfpathlineto{\pgfqpoint{1.376cm}{0cm}}
\pgfpathlineto{\pgfqpoint{1.376cm}{1.588cm}}
\pgfpathlineto{\pgfqpoint{0cm}{1.588cm}}
\pgfpathclose
\pgfusepath{fill}
\end{pgfscope}
\begin{pgfscope}
\pgfsetdash{}{0cm}
\pgfsetlinewidth{0.818mm}
\pgfsetroundcap
\pgfsetroundjoin
\pgfsetmiterlimit{7.0}
\definecolor{eps2pgf_color}{gray}{0}\pgfsetstrokecolor{eps2pgf_color}\pgfsetfillcolor{eps2pgf_color}
\pgfpathmoveto{\pgfqpoint{0.117cm}{1.476cm}}
\pgfpathlineto{\pgfqpoint{0.682cm}{0.726cm}}
\pgfpathlineto{\pgfqpoint{1.246cm}{1.476cm}}
\pgfusepath{stroke}
\end{pgfscope}
\definecolor{eps2pgf_color}{gray}{0}\pgfsetstrokecolor{eps2pgf_color}\pgfsetfillcolor{eps2pgf_color}
\pgfpathmoveto{\pgfqpoint{0.273cm}{1.451cm}}
\pgfpathcurveto{\pgfqpoint{0.273cm}{1.487cm}}{\pgfqpoint{0.259cm}{1.522cm}}{\pgfqpoint{0.233cm}{1.547cm}}
\pgfpathcurveto{\pgfqpoint{0.207cm}{1.573cm}}{\pgfqpoint{0.173cm}{1.588cm}}{\pgfqpoint{0.137cm}{1.588cm}}
\pgfpathcurveto{\pgfqpoint{0.1cm}{1.588cm}}{\pgfqpoint{0.066cm}{1.573cm}}{\pgfqpoint{0.04cm}{1.547cm}}
\pgfpathcurveto{\pgfqpoint{0.014cm}{1.522cm}}{\pgfqpoint{0cm}{1.487cm}}{\pgfqpoint{0cm}{1.451cm}}
\pgfpathcurveto{\pgfqpoint{0cm}{1.414cm}}{\pgfqpoint{0.014cm}{1.379cm}}{\pgfqpoint{0.04cm}{1.354cm}}
\pgfpathcurveto{\pgfqpoint{0.066cm}{1.328cm}}{\pgfqpoint{0.1cm}{1.314cm}}{\pgfqpoint{0.137cm}{1.314cm}}
\pgfpathcurveto{\pgfqpoint{0.173cm}{1.314cm}}{\pgfqpoint{0.207cm}{1.328cm}}{\pgfqpoint{0.233cm}{1.354cm}}
\pgfpathcurveto{\pgfqpoint{0.259cm}{1.379cm}}{\pgfqpoint{0.273cm}{1.414cm}}{\pgfqpoint{0.273cm}{1.451cm}}
\pgfusepath{fill}
\pgfpathmoveto{\pgfqpoint{1.345cm}{1.426cm}}
\pgfpathcurveto{\pgfqpoint{1.345cm}{1.463cm}}{\pgfqpoint{1.331cm}{1.497cm}}{\pgfqpoint{1.305cm}{1.523cm}}
\pgfpathcurveto{\pgfqpoint{1.28cm}{1.549cm}}{\pgfqpoint{1.245cm}{1.563cm}}{\pgfqpoint{1.209cm}{1.563cm}}
\pgfpathcurveto{\pgfqpoint{1.172cm}{1.563cm}}{\pgfqpoint{1.138cm}{1.549cm}}{\pgfqpoint{1.112cm}{1.523cm}}
\pgfpathcurveto{\pgfqpoint{1.087cm}{1.497cm}}{\pgfqpoint{1.072cm}{1.463cm}}{\pgfqpoint{1.072cm}{1.426cm}}
\pgfpathcurveto{\pgfqpoint{1.072cm}{1.39cm}}{\pgfqpoint{1.087cm}{1.355cm}}{\pgfqpoint{1.112cm}{1.329cm}}
\pgfpathcurveto{\pgfqpoint{1.138cm}{1.304cm}}{\pgfqpoint{1.172cm}{1.289cm}}{\pgfqpoint{1.209cm}{1.289cm}}
\pgfpathcurveto{\pgfqpoint{1.245cm}{1.289cm}}{\pgfqpoint{1.28cm}{1.304cm}}{\pgfqpoint{1.305cm}{1.329cm}}
\pgfpathcurveto{\pgfqpoint{1.331cm}{1.355cm}}{\pgfqpoint{1.345cm}{1.39cm}}{\pgfqpoint{1.345cm}{1.426cm}}
\pgfusepath{fill}
\begin{pgfscope}
\pgfsetdash{}{0cm}
\pgfsetlinewidth{0.818mm}
\pgfsetroundcap
\pgfsetmiterlimit{4.0}
\pgfpathmoveto{\pgfqpoint{0.682cm}{0.726cm}}
\pgfpathlineto{\pgfqpoint{0.682cm}{0.097cm}}
\pgfusepath{stroke}
\end{pgfscope}
\end{pgfscope}
\end{pgfscope}
\end{pgfscope}
\end{tikzpicture}}}+6(\nabla(-X^{\!\resizebox{0.6em}{!}{
\begin{tikzpicture}
\pgfpathmoveto{\pgfqpoint{0cm}{-0.035cm}}
\pgfpathlineto{\pgfqpoint{1.376cm}{-0.035cm}}
\pgfpathlineto{\pgfqpoint{1.376cm}{1.552cm}}
\pgfpathlineto{\pgfqpoint{0cm}{1.552cm}}
\pgfpathclose
\pgfusepath{clip}
\begin{pgfscope}
\begin{pgfscope}
\pgfpathmoveto{\pgfqpoint{0cm}{-0.035cm}}
\pgfpathlineto{\pgfqpoint{1.376cm}{-0.035cm}}
\pgfpathlineto{\pgfqpoint{1.376cm}{1.552cm}}
\pgfpathlineto{\pgfqpoint{0cm}{1.552cm}}
\pgfpathclose
\pgfusepath{clip}
\begin{pgfscope}
\begin{pgfscope}
\pgfsetdash{}{0cm}
\pgfsetlinewidth{0.818mm}
\pgfsetroundcap
\pgfsetroundjoin
\pgfsetmiterlimit{7.0}
\definecolor{eps2pgf_color}{gray}{0}\pgfsetstrokecolor{eps2pgf_color}\pgfsetfillcolor{eps2pgf_color}
\pgfpathmoveto{\pgfqpoint{0.117cm}{1.421cm}}
\pgfpathlineto{\pgfqpoint{0.682cm}{0.671cm}}
\pgfpathlineto{\pgfqpoint{1.246cm}{1.421cm}}
\pgfusepath{stroke}
\end{pgfscope}
\definecolor{eps2pgf_color}{gray}{0}\pgfsetstrokecolor{eps2pgf_color}\pgfsetfillcolor{eps2pgf_color}
\pgfpathmoveto{\pgfqpoint{0.273cm}{1.395cm}}
\pgfpathcurveto{\pgfqpoint{0.273cm}{1.432cm}}{\pgfqpoint{0.259cm}{1.467cm}}{\pgfqpoint{0.233cm}{1.492cm}}
\pgfpathcurveto{\pgfqpoint{0.207cm}{1.518cm}}{\pgfqpoint{0.173cm}{1.532cm}}{\pgfqpoint{0.137cm}{1.532cm}}
\pgfpathcurveto{\pgfqpoint{0.1cm}{1.532cm}}{\pgfqpoint{0.066cm}{1.518cm}}{\pgfqpoint{0.04cm}{1.492cm}}
\pgfpathcurveto{\pgfqpoint{0.014cm}{1.467cm}}{\pgfqpoint{0cm}{1.432cm}}{\pgfqpoint{0cm}{1.395cm}}
\pgfpathcurveto{\pgfqpoint{0cm}{1.359cm}}{\pgfqpoint{0.014cm}{1.324cm}}{\pgfqpoint{0.04cm}{1.299cm}}
\pgfpathcurveto{\pgfqpoint{0.066cm}{1.273cm}}{\pgfqpoint{0.1cm}{1.258cm}}{\pgfqpoint{0.137cm}{1.258cm}}
\pgfpathcurveto{\pgfqpoint{0.173cm}{1.258cm}}{\pgfqpoint{0.207cm}{1.273cm}}{\pgfqpoint{0.233cm}{1.299cm}}
\pgfpathcurveto{\pgfqpoint{0.259cm}{1.324cm}}{\pgfqpoint{0.273cm}{1.359cm}}{\pgfqpoint{0.273cm}{1.395cm}}
\pgfusepath{fill}
\begin{pgfscope}
\pgfsetdash{}{0cm}
\pgfsetlinewidth{0.818mm}
\pgfsetmiterlimit{7.0}
\pgfpathmoveto{\pgfqpoint{0.682cm}{0.671cm}}
\pgfpathlineto{\pgfqpoint{0.679cm}{1.418cm}}
\pgfusepath{stroke}
\end{pgfscope}
\pgfpathmoveto{\pgfqpoint{0.815cm}{1.399cm}}
\pgfpathcurveto{\pgfqpoint{0.815cm}{1.435cm}}{\pgfqpoint{0.801cm}{1.47cm}}{\pgfqpoint{0.775cm}{1.496cm}}
\pgfpathcurveto{\pgfqpoint{0.75cm}{1.521cm}}{\pgfqpoint{0.715cm}{1.536cm}}{\pgfqpoint{0.679cm}{1.536cm}}
\pgfpathcurveto{\pgfqpoint{0.643cm}{1.536cm}}{\pgfqpoint{0.608cm}{1.521cm}}{\pgfqpoint{0.582cm}{1.496cm}}
\pgfpathcurveto{\pgfqpoint{0.557cm}{1.47cm}}{\pgfqpoint{0.542cm}{1.435cm}}{\pgfqpoint{0.542cm}{1.399cm}}
\pgfpathcurveto{\pgfqpoint{0.542cm}{1.363cm}}{\pgfqpoint{0.557cm}{1.328cm}}{\pgfqpoint{0.582cm}{1.302cm}}
\pgfpathcurveto{\pgfqpoint{0.608cm}{1.276cm}}{\pgfqpoint{0.643cm}{1.262cm}}{\pgfqpoint{0.679cm}{1.262cm}}
\pgfpathcurveto{\pgfqpoint{0.715cm}{1.262cm}}{\pgfqpoint{0.75cm}{1.276cm}}{\pgfqpoint{0.775cm}{1.302cm}}
\pgfpathcurveto{\pgfqpoint{0.801cm}{1.328cm}}{\pgfqpoint{0.815cm}{1.363cm}}{\pgfqpoint{0.815cm}{1.399cm}}
\pgfusepath{fill}
\pgfpathmoveto{\pgfqpoint{1.345cm}{1.371cm}}
\pgfpathcurveto{\pgfqpoint{1.345cm}{1.408cm}}{\pgfqpoint{1.331cm}{1.442cm}}{\pgfqpoint{1.305cm}{1.468cm}}
\pgfpathcurveto{\pgfqpoint{1.28cm}{1.494cm}}{\pgfqpoint{1.245cm}{1.508cm}}{\pgfqpoint{1.209cm}{1.508cm}}
\pgfpathcurveto{\pgfqpoint{1.172cm}{1.508cm}}{\pgfqpoint{1.138cm}{1.494cm}}{\pgfqpoint{1.112cm}{1.468cm}}
\pgfpathcurveto{\pgfqpoint{1.087cm}{1.442cm}}{\pgfqpoint{1.072cm}{1.408cm}}{\pgfqpoint{1.072cm}{1.371cm}}
\pgfpathcurveto{\pgfqpoint{1.072cm}{1.335cm}}{\pgfqpoint{1.087cm}{1.3cm}}{\pgfqpoint{1.112cm}{1.274cm}}
\pgfpathcurveto{\pgfqpoint{1.138cm}{1.249cm}}{\pgfqpoint{1.172cm}{1.234cm}}{\pgfqpoint{1.209cm}{1.234cm}}
\pgfpathcurveto{\pgfqpoint{1.245cm}{1.234cm}}{\pgfqpoint{1.28cm}{1.249cm}}{\pgfqpoint{1.305cm}{1.274cm}}
\pgfpathcurveto{\pgfqpoint{1.331cm}{1.3cm}}{\pgfqpoint{1.345cm}{1.335cm}}{\pgfqpoint{1.345cm}{1.371cm}}
\pgfusepath{fill}
\begin{pgfscope}
\pgfsetdash{}{0cm}
\pgfsetlinewidth{0.818mm}
\pgfsetroundcap
\pgfsetmiterlimit{4.0}
\pgfpathmoveto{\pgfqpoint{0.682cm}{0.671cm}}
\pgfpathlineto{\pgfqpoint{0.682cm}{0.042cm}}
\pgfusepath{stroke}
\end{pgfscope}
\end{pgfscope}
\end{pgfscope}
\end{pgfscope}
\end{tikzpicture}}}+\phi+\psi))\prec(\nabla X^{\!\resizebox{0.6em}{!}{
\begin{tikzpicture}
\pgfpathmoveto{\pgfqpoint{0cm}{0cm}}
\pgfpathlineto{\pgfqpoint{1.376cm}{0cm}}
\pgfpathlineto{\pgfqpoint{1.376cm}{1.588cm}}
\pgfpathlineto{\pgfqpoint{0cm}{1.588cm}}
\pgfpathclose
\pgfusepath{clip}
\begin{pgfscope}
\begin{pgfscope}
\pgfpathmoveto{\pgfqpoint{0cm}{0cm}}
\pgfpathlineto{\pgfqpoint{1.376cm}{0cm}}
\pgfpathlineto{\pgfqpoint{1.376cm}{1.588cm}}
\pgfpathlineto{\pgfqpoint{0cm}{1.588cm}}
\pgfpathclose
\pgfusepath{clip}
\begin{pgfscope}
\begin{pgfscope}
\definecolor{eps2pgf_color}{gray}{0.976471}\pgfsetstrokecolor{eps2pgf_color}\pgfsetfillcolor{eps2pgf_color}
\pgfpathmoveto{\pgfqpoint{0cm}{0cm}}
\pgfpathlineto{\pgfqpoint{1.376cm}{0cm}}
\pgfpathlineto{\pgfqpoint{1.376cm}{1.588cm}}
\pgfpathlineto{\pgfqpoint{0cm}{1.588cm}}
\pgfpathclose
\pgfusepath{fill}
\end{pgfscope}
\begin{pgfscope}
\pgfsetdash{}{0cm}
\pgfsetlinewidth{0.818mm}
\pgfsetroundcap
\pgfsetroundjoin
\pgfsetmiterlimit{7.0}
\definecolor{eps2pgf_color}{gray}{0}\pgfsetstrokecolor{eps2pgf_color}\pgfsetfillcolor{eps2pgf_color}
\pgfpathmoveto{\pgfqpoint{0.117cm}{1.476cm}}
\pgfpathlineto{\pgfqpoint{0.682cm}{0.726cm}}
\pgfpathlineto{\pgfqpoint{1.246cm}{1.476cm}}
\pgfusepath{stroke}
\end{pgfscope}
\definecolor{eps2pgf_color}{gray}{0}\pgfsetstrokecolor{eps2pgf_color}\pgfsetfillcolor{eps2pgf_color}
\pgfpathmoveto{\pgfqpoint{0.273cm}{1.451cm}}
\pgfpathcurveto{\pgfqpoint{0.273cm}{1.487cm}}{\pgfqpoint{0.259cm}{1.522cm}}{\pgfqpoint{0.233cm}{1.547cm}}
\pgfpathcurveto{\pgfqpoint{0.207cm}{1.573cm}}{\pgfqpoint{0.173cm}{1.588cm}}{\pgfqpoint{0.137cm}{1.588cm}}
\pgfpathcurveto{\pgfqpoint{0.1cm}{1.588cm}}{\pgfqpoint{0.066cm}{1.573cm}}{\pgfqpoint{0.04cm}{1.547cm}}
\pgfpathcurveto{\pgfqpoint{0.014cm}{1.522cm}}{\pgfqpoint{0cm}{1.487cm}}{\pgfqpoint{0cm}{1.451cm}}
\pgfpathcurveto{\pgfqpoint{0cm}{1.414cm}}{\pgfqpoint{0.014cm}{1.379cm}}{\pgfqpoint{0.04cm}{1.354cm}}
\pgfpathcurveto{\pgfqpoint{0.066cm}{1.328cm}}{\pgfqpoint{0.1cm}{1.314cm}}{\pgfqpoint{0.137cm}{1.314cm}}
\pgfpathcurveto{\pgfqpoint{0.173cm}{1.314cm}}{\pgfqpoint{0.207cm}{1.328cm}}{\pgfqpoint{0.233cm}{1.354cm}}
\pgfpathcurveto{\pgfqpoint{0.259cm}{1.379cm}}{\pgfqpoint{0.273cm}{1.414cm}}{\pgfqpoint{0.273cm}{1.451cm}}
\pgfusepath{fill}
\pgfpathmoveto{\pgfqpoint{1.345cm}{1.426cm}}
\pgfpathcurveto{\pgfqpoint{1.345cm}{1.463cm}}{\pgfqpoint{1.331cm}{1.497cm}}{\pgfqpoint{1.305cm}{1.523cm}}
\pgfpathcurveto{\pgfqpoint{1.28cm}{1.549cm}}{\pgfqpoint{1.245cm}{1.563cm}}{\pgfqpoint{1.209cm}{1.563cm}}
\pgfpathcurveto{\pgfqpoint{1.172cm}{1.563cm}}{\pgfqpoint{1.138cm}{1.549cm}}{\pgfqpoint{1.112cm}{1.523cm}}
\pgfpathcurveto{\pgfqpoint{1.087cm}{1.497cm}}{\pgfqpoint{1.072cm}{1.463cm}}{\pgfqpoint{1.072cm}{1.426cm}}
\pgfpathcurveto{\pgfqpoint{1.072cm}{1.39cm}}{\pgfqpoint{1.087cm}{1.355cm}}{\pgfqpoint{1.112cm}{1.329cm}}
\pgfpathcurveto{\pgfqpoint{1.138cm}{1.304cm}}{\pgfqpoint{1.172cm}{1.289cm}}{\pgfqpoint{1.209cm}{1.289cm}}
\pgfpathcurveto{\pgfqpoint{1.245cm}{1.289cm}}{\pgfqpoint{1.28cm}{1.304cm}}{\pgfqpoint{1.305cm}{1.329cm}}
\pgfpathcurveto{\pgfqpoint{1.331cm}{1.355cm}}{\pgfqpoint{1.345cm}{1.39cm}}{\pgfqpoint{1.345cm}{1.426cm}}
\pgfusepath{fill}
\begin{pgfscope}
\pgfsetdash{}{0cm}
\pgfsetlinewidth{0.818mm}
\pgfsetroundcap
\pgfsetmiterlimit{4.0}
\pgfpathmoveto{\pgfqpoint{0.682cm}{0.726cm}}
\pgfpathlineto{\pgfqpoint{0.682cm}{0.097cm}}
\pgfusepath{stroke}
\end{pgfscope}
\end{pgfscope}
\end{pgfscope}
\end{pgfscope}
\end{tikzpicture}}})\|_{\CC^{-1+\alpha}(\rho^{2+\alpha})}\\
&\quad\lesssim \|-X^{\!\resizebox{0.6em}{!}{
\begin{tikzpicture}
\pgfpathmoveto{\pgfqpoint{0cm}{-0.035cm}}
\pgfpathlineto{\pgfqpoint{1.376cm}{-0.035cm}}
\pgfpathlineto{\pgfqpoint{1.376cm}{1.552cm}}
\pgfpathlineto{\pgfqpoint{0cm}{1.552cm}}
\pgfpathclose
\pgfusepath{clip}
\begin{pgfscope}
\begin{pgfscope}
\pgfpathmoveto{\pgfqpoint{0cm}{-0.035cm}}
\pgfpathlineto{\pgfqpoint{1.376cm}{-0.035cm}}
\pgfpathlineto{\pgfqpoint{1.376cm}{1.552cm}}
\pgfpathlineto{\pgfqpoint{0cm}{1.552cm}}
\pgfpathclose
\pgfusepath{clip}
\begin{pgfscope}
\begin{pgfscope}
\pgfsetdash{}{0cm}
\pgfsetlinewidth{0.818mm}
\pgfsetroundcap
\pgfsetroundjoin
\pgfsetmiterlimit{7.0}
\definecolor{eps2pgf_color}{gray}{0}\pgfsetstrokecolor{eps2pgf_color}\pgfsetfillcolor{eps2pgf_color}
\pgfpathmoveto{\pgfqpoint{0.117cm}{1.421cm}}
\pgfpathlineto{\pgfqpoint{0.682cm}{0.671cm}}
\pgfpathlineto{\pgfqpoint{1.246cm}{1.421cm}}
\pgfusepath{stroke}
\end{pgfscope}
\definecolor{eps2pgf_color}{gray}{0}\pgfsetstrokecolor{eps2pgf_color}\pgfsetfillcolor{eps2pgf_color}
\pgfpathmoveto{\pgfqpoint{0.273cm}{1.395cm}}
\pgfpathcurveto{\pgfqpoint{0.273cm}{1.432cm}}{\pgfqpoint{0.259cm}{1.467cm}}{\pgfqpoint{0.233cm}{1.492cm}}
\pgfpathcurveto{\pgfqpoint{0.207cm}{1.518cm}}{\pgfqpoint{0.173cm}{1.532cm}}{\pgfqpoint{0.137cm}{1.532cm}}
\pgfpathcurveto{\pgfqpoint{0.1cm}{1.532cm}}{\pgfqpoint{0.066cm}{1.518cm}}{\pgfqpoint{0.04cm}{1.492cm}}
\pgfpathcurveto{\pgfqpoint{0.014cm}{1.467cm}}{\pgfqpoint{0cm}{1.432cm}}{\pgfqpoint{0cm}{1.395cm}}
\pgfpathcurveto{\pgfqpoint{0cm}{1.359cm}}{\pgfqpoint{0.014cm}{1.324cm}}{\pgfqpoint{0.04cm}{1.299cm}}
\pgfpathcurveto{\pgfqpoint{0.066cm}{1.273cm}}{\pgfqpoint{0.1cm}{1.258cm}}{\pgfqpoint{0.137cm}{1.258cm}}
\pgfpathcurveto{\pgfqpoint{0.173cm}{1.258cm}}{\pgfqpoint{0.207cm}{1.273cm}}{\pgfqpoint{0.233cm}{1.299cm}}
\pgfpathcurveto{\pgfqpoint{0.259cm}{1.324cm}}{\pgfqpoint{0.273cm}{1.359cm}}{\pgfqpoint{0.273cm}{1.395cm}}
\pgfusepath{fill}
\begin{pgfscope}
\pgfsetdash{}{0cm}
\pgfsetlinewidth{0.818mm}
\pgfsetmiterlimit{7.0}
\pgfpathmoveto{\pgfqpoint{0.682cm}{0.671cm}}
\pgfpathlineto{\pgfqpoint{0.679cm}{1.418cm}}
\pgfusepath{stroke}
\end{pgfscope}
\pgfpathmoveto{\pgfqpoint{0.815cm}{1.399cm}}
\pgfpathcurveto{\pgfqpoint{0.815cm}{1.435cm}}{\pgfqpoint{0.801cm}{1.47cm}}{\pgfqpoint{0.775cm}{1.496cm}}
\pgfpathcurveto{\pgfqpoint{0.75cm}{1.521cm}}{\pgfqpoint{0.715cm}{1.536cm}}{\pgfqpoint{0.679cm}{1.536cm}}
\pgfpathcurveto{\pgfqpoint{0.643cm}{1.536cm}}{\pgfqpoint{0.608cm}{1.521cm}}{\pgfqpoint{0.582cm}{1.496cm}}
\pgfpathcurveto{\pgfqpoint{0.557cm}{1.47cm}}{\pgfqpoint{0.542cm}{1.435cm}}{\pgfqpoint{0.542cm}{1.399cm}}
\pgfpathcurveto{\pgfqpoint{0.542cm}{1.363cm}}{\pgfqpoint{0.557cm}{1.328cm}}{\pgfqpoint{0.582cm}{1.302cm}}
\pgfpathcurveto{\pgfqpoint{0.608cm}{1.276cm}}{\pgfqpoint{0.643cm}{1.262cm}}{\pgfqpoint{0.679cm}{1.262cm}}
\pgfpathcurveto{\pgfqpoint{0.715cm}{1.262cm}}{\pgfqpoint{0.75cm}{1.276cm}}{\pgfqpoint{0.775cm}{1.302cm}}
\pgfpathcurveto{\pgfqpoint{0.801cm}{1.328cm}}{\pgfqpoint{0.815cm}{1.363cm}}{\pgfqpoint{0.815cm}{1.399cm}}
\pgfusepath{fill}
\pgfpathmoveto{\pgfqpoint{1.345cm}{1.371cm}}
\pgfpathcurveto{\pgfqpoint{1.345cm}{1.408cm}}{\pgfqpoint{1.331cm}{1.442cm}}{\pgfqpoint{1.305cm}{1.468cm}}
\pgfpathcurveto{\pgfqpoint{1.28cm}{1.494cm}}{\pgfqpoint{1.245cm}{1.508cm}}{\pgfqpoint{1.209cm}{1.508cm}}
\pgfpathcurveto{\pgfqpoint{1.172cm}{1.508cm}}{\pgfqpoint{1.138cm}{1.494cm}}{\pgfqpoint{1.112cm}{1.468cm}}
\pgfpathcurveto{\pgfqpoint{1.087cm}{1.442cm}}{\pgfqpoint{1.072cm}{1.408cm}}{\pgfqpoint{1.072cm}{1.371cm}}
\pgfpathcurveto{\pgfqpoint{1.072cm}{1.335cm}}{\pgfqpoint{1.087cm}{1.3cm}}{\pgfqpoint{1.112cm}{1.274cm}}
\pgfpathcurveto{\pgfqpoint{1.138cm}{1.249cm}}{\pgfqpoint{1.172cm}{1.234cm}}{\pgfqpoint{1.209cm}{1.234cm}}
\pgfpathcurveto{\pgfqpoint{1.245cm}{1.234cm}}{\pgfqpoint{1.28cm}{1.249cm}}{\pgfqpoint{1.305cm}{1.274cm}}
\pgfpathcurveto{\pgfqpoint{1.331cm}{1.3cm}}{\pgfqpoint{1.345cm}{1.335cm}}{\pgfqpoint{1.345cm}{1.371cm}}
\pgfusepath{fill}
\begin{pgfscope}
\pgfsetdash{}{0cm}
\pgfsetlinewidth{0.818mm}
\pgfsetroundcap
\pgfsetmiterlimit{4.0}
\pgfpathmoveto{\pgfqpoint{0.682cm}{0.671cm}}
\pgfpathlineto{\pgfqpoint{0.682cm}{0.042cm}}
\pgfusepath{stroke}
\end{pgfscope}
\end{pgfscope}
\end{pgfscope}
\end{pgfscope}
\end{tikzpicture}}}+\phi+\psi\|_{\CC^{\frac{1}{2}-\kappa}(\rho^{\frac{3}{2}+\alpha})}\| X^{\!\resizebox{0.6em}{!}{
\begin{tikzpicture}
\pgfpathmoveto{\pgfqpoint{0cm}{0cm}}
\pgfpathlineto{\pgfqpoint{1.376cm}{0cm}}
\pgfpathlineto{\pgfqpoint{1.376cm}{1.588cm}}
\pgfpathlineto{\pgfqpoint{0cm}{1.588cm}}
\pgfpathclose
\pgfusepath{clip}
\begin{pgfscope}
\begin{pgfscope}
\pgfpathmoveto{\pgfqpoint{0cm}{0cm}}
\pgfpathlineto{\pgfqpoint{1.376cm}{0cm}}
\pgfpathlineto{\pgfqpoint{1.376cm}{1.588cm}}
\pgfpathlineto{\pgfqpoint{0cm}{1.588cm}}
\pgfpathclose
\pgfusepath{clip}
\begin{pgfscope}
\begin{pgfscope}
\definecolor{eps2pgf_color}{gray}{0.976471}\pgfsetstrokecolor{eps2pgf_color}\pgfsetfillcolor{eps2pgf_color}
\pgfpathmoveto{\pgfqpoint{0cm}{0cm}}
\pgfpathlineto{\pgfqpoint{1.376cm}{0cm}}
\pgfpathlineto{\pgfqpoint{1.376cm}{1.588cm}}
\pgfpathlineto{\pgfqpoint{0cm}{1.588cm}}
\pgfpathclose
\pgfusepath{fill}
\end{pgfscope}
\begin{pgfscope}
\pgfsetdash{}{0cm}
\pgfsetlinewidth{0.818mm}
\pgfsetroundcap
\pgfsetroundjoin
\pgfsetmiterlimit{7.0}
\definecolor{eps2pgf_color}{gray}{0}\pgfsetstrokecolor{eps2pgf_color}\pgfsetfillcolor{eps2pgf_color}
\pgfpathmoveto{\pgfqpoint{0.117cm}{1.476cm}}
\pgfpathlineto{\pgfqpoint{0.682cm}{0.726cm}}
\pgfpathlineto{\pgfqpoint{1.246cm}{1.476cm}}
\pgfusepath{stroke}
\end{pgfscope}
\definecolor{eps2pgf_color}{gray}{0}\pgfsetstrokecolor{eps2pgf_color}\pgfsetfillcolor{eps2pgf_color}
\pgfpathmoveto{\pgfqpoint{0.273cm}{1.451cm}}
\pgfpathcurveto{\pgfqpoint{0.273cm}{1.487cm}}{\pgfqpoint{0.259cm}{1.522cm}}{\pgfqpoint{0.233cm}{1.547cm}}
\pgfpathcurveto{\pgfqpoint{0.207cm}{1.573cm}}{\pgfqpoint{0.173cm}{1.588cm}}{\pgfqpoint{0.137cm}{1.588cm}}
\pgfpathcurveto{\pgfqpoint{0.1cm}{1.588cm}}{\pgfqpoint{0.066cm}{1.573cm}}{\pgfqpoint{0.04cm}{1.547cm}}
\pgfpathcurveto{\pgfqpoint{0.014cm}{1.522cm}}{\pgfqpoint{0cm}{1.487cm}}{\pgfqpoint{0cm}{1.451cm}}
\pgfpathcurveto{\pgfqpoint{0cm}{1.414cm}}{\pgfqpoint{0.014cm}{1.379cm}}{\pgfqpoint{0.04cm}{1.354cm}}
\pgfpathcurveto{\pgfqpoint{0.066cm}{1.328cm}}{\pgfqpoint{0.1cm}{1.314cm}}{\pgfqpoint{0.137cm}{1.314cm}}
\pgfpathcurveto{\pgfqpoint{0.173cm}{1.314cm}}{\pgfqpoint{0.207cm}{1.328cm}}{\pgfqpoint{0.233cm}{1.354cm}}
\pgfpathcurveto{\pgfqpoint{0.259cm}{1.379cm}}{\pgfqpoint{0.273cm}{1.414cm}}{\pgfqpoint{0.273cm}{1.451cm}}
\pgfusepath{fill}
\pgfpathmoveto{\pgfqpoint{1.345cm}{1.426cm}}
\pgfpathcurveto{\pgfqpoint{1.345cm}{1.463cm}}{\pgfqpoint{1.331cm}{1.497cm}}{\pgfqpoint{1.305cm}{1.523cm}}
\pgfpathcurveto{\pgfqpoint{1.28cm}{1.549cm}}{\pgfqpoint{1.245cm}{1.563cm}}{\pgfqpoint{1.209cm}{1.563cm}}
\pgfpathcurveto{\pgfqpoint{1.172cm}{1.563cm}}{\pgfqpoint{1.138cm}{1.549cm}}{\pgfqpoint{1.112cm}{1.523cm}}
\pgfpathcurveto{\pgfqpoint{1.087cm}{1.497cm}}{\pgfqpoint{1.072cm}{1.463cm}}{\pgfqpoint{1.072cm}{1.426cm}}
\pgfpathcurveto{\pgfqpoint{1.072cm}{1.39cm}}{\pgfqpoint{1.087cm}{1.355cm}}{\pgfqpoint{1.112cm}{1.329cm}}
\pgfpathcurveto{\pgfqpoint{1.138cm}{1.304cm}}{\pgfqpoint{1.172cm}{1.289cm}}{\pgfqpoint{1.209cm}{1.289cm}}
\pgfpathcurveto{\pgfqpoint{1.245cm}{1.289cm}}{\pgfqpoint{1.28cm}{1.304cm}}{\pgfqpoint{1.305cm}{1.329cm}}
\pgfpathcurveto{\pgfqpoint{1.331cm}{1.355cm}}{\pgfqpoint{1.345cm}{1.39cm}}{\pgfqpoint{1.345cm}{1.426cm}}
\pgfusepath{fill}
\begin{pgfscope}
\pgfsetdash{}{0cm}
\pgfsetlinewidth{0.818mm}
\pgfsetroundcap
\pgfsetmiterlimit{4.0}
\pgfpathmoveto{\pgfqpoint{0.682cm}{0.726cm}}
\pgfpathlineto{\pgfqpoint{0.682cm}{0.097cm}}
\pgfusepath{stroke}
\end{pgfscope}
\end{pgfscope}
\end{pgfscope}
\end{pgfscope}
\end{tikzpicture}}}\|_{\CC^{1-\kappa}(\rho^\sigma)}\lesssim 1+\|\phi+\psi\|_{\CC^{\frac{1}{2}+\alpha}(\rho^{\frac{3}{2}+\alpha})}\\
&\quad \lesssim 1+\|\psi\|^\varepsilon_{L^\infty(\rho)}+\|\psi\|_{\CC^{\frac{1}{2}+\alpha}(\rho^{\frac{3}{2}+\alpha})}.
\end{align*}

To summarize, we have proved that
\begin{equation}\label{eq:12}
\|\vartheta\|_{\CC^{1+\alpha}(\rho^{2+\alpha})}\lesssim 1 +  \|\psi\|^{1+\varepsilon}_{L^\infty(\rho)}+\|\psi\|_{\CC^{\frac{1}{2}+\alpha}(\rho^{\frac{3}{2}+\alpha})}.
\end{equation}

\subsection{Bound for $\psi$ in $\CC^{2+\gamma}(\rho^{3+\gamma})$}
\label{ssec:psi1}

In this section we make use of the estimates \eqref{eq:10}, \eqref{eq:11}, \eqref{eq:12} in order to estimate  $\psi$ in $\CC^{2+\gamma}(\rho^{3+\gamma})$ for  $\gamma=\alpha-\kappa>0$ such that $\gamma\leq\frac{1}{2}-3\kappa$ (which can be achieved by a suitable choice of $\alpha, \kappa>0$. In view of \eqref{eq:two} and Lemma \ref{lemma:schauder-ellptic}, it is therefore necessary to estimate $\Psi$ in $\CC^{\gamma}(\rho^{3+\gamma})$.
We estimate as follows
\begin{align*}
\|3\llbracket X^2\rrbracket\circ\psi\|_{\CC^\gamma(\rho^{3+\gamma})}+\|3\llbracket X^2\rrbracket\circ\vartheta\|_{\CC^\gamma(\rho^{3+\gamma})}&\lesssim \|\psi\|_{\CC^{1+\alpha}(\rho^{2+\alpha})}+\|\vartheta\|_{\CC^{1+\alpha}(\rho^{2+\alpha})}\\
&\lesssim\|\psi\|_{\CC^{1+\alpha}(\rho^{2+\alpha})}+1+\|\psi\|^{1+\varepsilon}_{L^\infty(\rho)}+\|\psi\|_{\CC^{\frac{1}{2}+\alpha}(\rho^{\frac{3}{2}+\alpha})} ,
\end{align*}
and according to Lemma \ref{lem:com}
\begin{align*}
\|9\mathrm{com}(-X^{\!\resizebox{0.6em}{!}{
\begin{tikzpicture}
\pgfpathmoveto{\pgfqpoint{0cm}{-0.035cm}}
\pgfpathlineto{\pgfqpoint{1.376cm}{-0.035cm}}
\pgfpathlineto{\pgfqpoint{1.376cm}{1.552cm}}
\pgfpathlineto{\pgfqpoint{0cm}{1.552cm}}
\pgfpathclose
\pgfusepath{clip}
\begin{pgfscope}
\begin{pgfscope}
\pgfpathmoveto{\pgfqpoint{0cm}{-0.035cm}}
\pgfpathlineto{\pgfqpoint{1.376cm}{-0.035cm}}
\pgfpathlineto{\pgfqpoint{1.376cm}{1.552cm}}
\pgfpathlineto{\pgfqpoint{0cm}{1.552cm}}
\pgfpathclose
\pgfusepath{clip}
\begin{pgfscope}
\begin{pgfscope}
\pgfsetdash{}{0cm}
\pgfsetlinewidth{0.818mm}
\pgfsetroundcap
\pgfsetroundjoin
\pgfsetmiterlimit{7.0}
\definecolor{eps2pgf_color}{gray}{0}\pgfsetstrokecolor{eps2pgf_color}\pgfsetfillcolor{eps2pgf_color}
\pgfpathmoveto{\pgfqpoint{0.117cm}{1.421cm}}
\pgfpathlineto{\pgfqpoint{0.682cm}{0.671cm}}
\pgfpathlineto{\pgfqpoint{1.246cm}{1.421cm}}
\pgfusepath{stroke}
\end{pgfscope}
\definecolor{eps2pgf_color}{gray}{0}\pgfsetstrokecolor{eps2pgf_color}\pgfsetfillcolor{eps2pgf_color}
\pgfpathmoveto{\pgfqpoint{0.273cm}{1.395cm}}
\pgfpathcurveto{\pgfqpoint{0.273cm}{1.432cm}}{\pgfqpoint{0.259cm}{1.467cm}}{\pgfqpoint{0.233cm}{1.492cm}}
\pgfpathcurveto{\pgfqpoint{0.207cm}{1.518cm}}{\pgfqpoint{0.173cm}{1.532cm}}{\pgfqpoint{0.137cm}{1.532cm}}
\pgfpathcurveto{\pgfqpoint{0.1cm}{1.532cm}}{\pgfqpoint{0.066cm}{1.518cm}}{\pgfqpoint{0.04cm}{1.492cm}}
\pgfpathcurveto{\pgfqpoint{0.014cm}{1.467cm}}{\pgfqpoint{0cm}{1.432cm}}{\pgfqpoint{0cm}{1.395cm}}
\pgfpathcurveto{\pgfqpoint{0cm}{1.359cm}}{\pgfqpoint{0.014cm}{1.324cm}}{\pgfqpoint{0.04cm}{1.299cm}}
\pgfpathcurveto{\pgfqpoint{0.066cm}{1.273cm}}{\pgfqpoint{0.1cm}{1.258cm}}{\pgfqpoint{0.137cm}{1.258cm}}
\pgfpathcurveto{\pgfqpoint{0.173cm}{1.258cm}}{\pgfqpoint{0.207cm}{1.273cm}}{\pgfqpoint{0.233cm}{1.299cm}}
\pgfpathcurveto{\pgfqpoint{0.259cm}{1.324cm}}{\pgfqpoint{0.273cm}{1.359cm}}{\pgfqpoint{0.273cm}{1.395cm}}
\pgfusepath{fill}
\begin{pgfscope}
\pgfsetdash{}{0cm}
\pgfsetlinewidth{0.818mm}
\pgfsetmiterlimit{7.0}
\pgfpathmoveto{\pgfqpoint{0.682cm}{0.671cm}}
\pgfpathlineto{\pgfqpoint{0.679cm}{1.418cm}}
\pgfusepath{stroke}
\end{pgfscope}
\pgfpathmoveto{\pgfqpoint{0.815cm}{1.399cm}}
\pgfpathcurveto{\pgfqpoint{0.815cm}{1.435cm}}{\pgfqpoint{0.801cm}{1.47cm}}{\pgfqpoint{0.775cm}{1.496cm}}
\pgfpathcurveto{\pgfqpoint{0.75cm}{1.521cm}}{\pgfqpoint{0.715cm}{1.536cm}}{\pgfqpoint{0.679cm}{1.536cm}}
\pgfpathcurveto{\pgfqpoint{0.643cm}{1.536cm}}{\pgfqpoint{0.608cm}{1.521cm}}{\pgfqpoint{0.582cm}{1.496cm}}
\pgfpathcurveto{\pgfqpoint{0.557cm}{1.47cm}}{\pgfqpoint{0.542cm}{1.435cm}}{\pgfqpoint{0.542cm}{1.399cm}}
\pgfpathcurveto{\pgfqpoint{0.542cm}{1.363cm}}{\pgfqpoint{0.557cm}{1.328cm}}{\pgfqpoint{0.582cm}{1.302cm}}
\pgfpathcurveto{\pgfqpoint{0.608cm}{1.276cm}}{\pgfqpoint{0.643cm}{1.262cm}}{\pgfqpoint{0.679cm}{1.262cm}}
\pgfpathcurveto{\pgfqpoint{0.715cm}{1.262cm}}{\pgfqpoint{0.75cm}{1.276cm}}{\pgfqpoint{0.775cm}{1.302cm}}
\pgfpathcurveto{\pgfqpoint{0.801cm}{1.328cm}}{\pgfqpoint{0.815cm}{1.363cm}}{\pgfqpoint{0.815cm}{1.399cm}}
\pgfusepath{fill}
\pgfpathmoveto{\pgfqpoint{1.345cm}{1.371cm}}
\pgfpathcurveto{\pgfqpoint{1.345cm}{1.408cm}}{\pgfqpoint{1.331cm}{1.442cm}}{\pgfqpoint{1.305cm}{1.468cm}}
\pgfpathcurveto{\pgfqpoint{1.28cm}{1.494cm}}{\pgfqpoint{1.245cm}{1.508cm}}{\pgfqpoint{1.209cm}{1.508cm}}
\pgfpathcurveto{\pgfqpoint{1.172cm}{1.508cm}}{\pgfqpoint{1.138cm}{1.494cm}}{\pgfqpoint{1.112cm}{1.468cm}}
\pgfpathcurveto{\pgfqpoint{1.087cm}{1.442cm}}{\pgfqpoint{1.072cm}{1.408cm}}{\pgfqpoint{1.072cm}{1.371cm}}
\pgfpathcurveto{\pgfqpoint{1.072cm}{1.335cm}}{\pgfqpoint{1.087cm}{1.3cm}}{\pgfqpoint{1.112cm}{1.274cm}}
\pgfpathcurveto{\pgfqpoint{1.138cm}{1.249cm}}{\pgfqpoint{1.172cm}{1.234cm}}{\pgfqpoint{1.209cm}{1.234cm}}
\pgfpathcurveto{\pgfqpoint{1.245cm}{1.234cm}}{\pgfqpoint{1.28cm}{1.249cm}}{\pgfqpoint{1.305cm}{1.274cm}}
\pgfpathcurveto{\pgfqpoint{1.331cm}{1.3cm}}{\pgfqpoint{1.345cm}{1.335cm}}{\pgfqpoint{1.345cm}{1.371cm}}
\pgfusepath{fill}
\begin{pgfscope}
\pgfsetdash{}{0cm}
\pgfsetlinewidth{0.818mm}
\pgfsetroundcap
\pgfsetmiterlimit{4.0}
\pgfpathmoveto{\pgfqpoint{0.682cm}{0.671cm}}
\pgfpathlineto{\pgfqpoint{0.682cm}{0.042cm}}
\pgfusepath{stroke}
\end{pgfscope}
\end{pgfscope}
\end{pgfscope}
\end{pgfscope}
\end{tikzpicture}}} + \phi + \psi,X^{\!\resizebox{0.6em}{!}{
\begin{tikzpicture}
\pgfpathmoveto{\pgfqpoint{0cm}{0cm}}
\pgfpathlineto{\pgfqpoint{1.376cm}{0cm}}
\pgfpathlineto{\pgfqpoint{1.376cm}{1.588cm}}
\pgfpathlineto{\pgfqpoint{0cm}{1.588cm}}
\pgfpathclose
\pgfusepath{clip}
\begin{pgfscope}
\begin{pgfscope}
\pgfpathmoveto{\pgfqpoint{0cm}{0cm}}
\pgfpathlineto{\pgfqpoint{1.376cm}{0cm}}
\pgfpathlineto{\pgfqpoint{1.376cm}{1.588cm}}
\pgfpathlineto{\pgfqpoint{0cm}{1.588cm}}
\pgfpathclose
\pgfusepath{clip}
\begin{pgfscope}
\begin{pgfscope}
\definecolor{eps2pgf_color}{gray}{0.976471}\pgfsetstrokecolor{eps2pgf_color}\pgfsetfillcolor{eps2pgf_color}
\pgfpathmoveto{\pgfqpoint{0cm}{0cm}}
\pgfpathlineto{\pgfqpoint{1.376cm}{0cm}}
\pgfpathlineto{\pgfqpoint{1.376cm}{1.588cm}}
\pgfpathlineto{\pgfqpoint{0cm}{1.588cm}}
\pgfpathclose
\pgfusepath{fill}
\end{pgfscope}
\begin{pgfscope}
\pgfsetdash{}{0cm}
\pgfsetlinewidth{0.818mm}
\pgfsetroundcap
\pgfsetroundjoin
\pgfsetmiterlimit{7.0}
\definecolor{eps2pgf_color}{gray}{0}\pgfsetstrokecolor{eps2pgf_color}\pgfsetfillcolor{eps2pgf_color}
\pgfpathmoveto{\pgfqpoint{0.117cm}{1.476cm}}
\pgfpathlineto{\pgfqpoint{0.682cm}{0.726cm}}
\pgfpathlineto{\pgfqpoint{1.246cm}{1.476cm}}
\pgfusepath{stroke}
\end{pgfscope}
\definecolor{eps2pgf_color}{gray}{0}\pgfsetstrokecolor{eps2pgf_color}\pgfsetfillcolor{eps2pgf_color}
\pgfpathmoveto{\pgfqpoint{0.273cm}{1.451cm}}
\pgfpathcurveto{\pgfqpoint{0.273cm}{1.487cm}}{\pgfqpoint{0.259cm}{1.522cm}}{\pgfqpoint{0.233cm}{1.547cm}}
\pgfpathcurveto{\pgfqpoint{0.207cm}{1.573cm}}{\pgfqpoint{0.173cm}{1.588cm}}{\pgfqpoint{0.137cm}{1.588cm}}
\pgfpathcurveto{\pgfqpoint{0.1cm}{1.588cm}}{\pgfqpoint{0.066cm}{1.573cm}}{\pgfqpoint{0.04cm}{1.547cm}}
\pgfpathcurveto{\pgfqpoint{0.014cm}{1.522cm}}{\pgfqpoint{0cm}{1.487cm}}{\pgfqpoint{0cm}{1.451cm}}
\pgfpathcurveto{\pgfqpoint{0cm}{1.414cm}}{\pgfqpoint{0.014cm}{1.379cm}}{\pgfqpoint{0.04cm}{1.354cm}}
\pgfpathcurveto{\pgfqpoint{0.066cm}{1.328cm}}{\pgfqpoint{0.1cm}{1.314cm}}{\pgfqpoint{0.137cm}{1.314cm}}
\pgfpathcurveto{\pgfqpoint{0.173cm}{1.314cm}}{\pgfqpoint{0.207cm}{1.328cm}}{\pgfqpoint{0.233cm}{1.354cm}}
\pgfpathcurveto{\pgfqpoint{0.259cm}{1.379cm}}{\pgfqpoint{0.273cm}{1.414cm}}{\pgfqpoint{0.273cm}{1.451cm}}
\pgfusepath{fill}
\pgfpathmoveto{\pgfqpoint{1.345cm}{1.426cm}}
\pgfpathcurveto{\pgfqpoint{1.345cm}{1.463cm}}{\pgfqpoint{1.331cm}{1.497cm}}{\pgfqpoint{1.305cm}{1.523cm}}
\pgfpathcurveto{\pgfqpoint{1.28cm}{1.549cm}}{\pgfqpoint{1.245cm}{1.563cm}}{\pgfqpoint{1.209cm}{1.563cm}}
\pgfpathcurveto{\pgfqpoint{1.172cm}{1.563cm}}{\pgfqpoint{1.138cm}{1.549cm}}{\pgfqpoint{1.112cm}{1.523cm}}
\pgfpathcurveto{\pgfqpoint{1.087cm}{1.497cm}}{\pgfqpoint{1.072cm}{1.463cm}}{\pgfqpoint{1.072cm}{1.426cm}}
\pgfpathcurveto{\pgfqpoint{1.072cm}{1.39cm}}{\pgfqpoint{1.087cm}{1.355cm}}{\pgfqpoint{1.112cm}{1.329cm}}
\pgfpathcurveto{\pgfqpoint{1.138cm}{1.304cm}}{\pgfqpoint{1.172cm}{1.289cm}}{\pgfqpoint{1.209cm}{1.289cm}}
\pgfpathcurveto{\pgfqpoint{1.245cm}{1.289cm}}{\pgfqpoint{1.28cm}{1.304cm}}{\pgfqpoint{1.305cm}{1.329cm}}
\pgfpathcurveto{\pgfqpoint{1.331cm}{1.355cm}}{\pgfqpoint{1.345cm}{1.39cm}}{\pgfqpoint{1.345cm}{1.426cm}}
\pgfusepath{fill}
\begin{pgfscope}
\pgfsetdash{}{0cm}
\pgfsetlinewidth{0.818mm}
\pgfsetroundcap
\pgfsetmiterlimit{4.0}
\pgfpathmoveto{\pgfqpoint{0.682cm}{0.726cm}}
\pgfpathlineto{\pgfqpoint{0.682cm}{0.097cm}}
\pgfusepath{stroke}
\end{pgfscope}
\end{pgfscope}
\end{pgfscope}
\end{pgfscope}
\end{tikzpicture}}},\llbracket X^2 \rrbracket)\|_{\CC^\gamma(\rho^{3+\gamma})}&\lesssim 1+\|\phi+\psi\|_{\CC^{\frac{1}{2}}(\rho^{2+\alpha})}\lesssim 1+\|\psi\|_{L^\infty(\rho)}^\varepsilon+\|\psi\|_{\CC^{{1+\alpha}}(\rho^{2+\alpha})},
\end{align*}
\begin{align*}
\|6\mathrm{com}(X^{\!\resizebox{0.6em}{!}{
\begin{tikzpicture}
\pgfpathmoveto{\pgfqpoint{0cm}{-0.035cm}}
\pgfpathlineto{\pgfqpoint{1.376cm}{-0.035cm}}
\pgfpathlineto{\pgfqpoint{1.376cm}{1.552cm}}
\pgfpathlineto{\pgfqpoint{0cm}{1.552cm}}
\pgfpathclose
\pgfusepath{clip}
\begin{pgfscope}
\begin{pgfscope}
\pgfpathmoveto{\pgfqpoint{0cm}{-0.035cm}}
\pgfpathlineto{\pgfqpoint{1.376cm}{-0.035cm}}
\pgfpathlineto{\pgfqpoint{1.376cm}{1.552cm}}
\pgfpathlineto{\pgfqpoint{0cm}{1.552cm}}
\pgfpathclose
\pgfusepath{clip}
\begin{pgfscope}
\begin{pgfscope}
\pgfsetdash{}{0cm}
\pgfsetlinewidth{0.818mm}
\pgfsetroundcap
\pgfsetroundjoin
\pgfsetmiterlimit{7.0}
\definecolor{eps2pgf_color}{gray}{0}\pgfsetstrokecolor{eps2pgf_color}\pgfsetfillcolor{eps2pgf_color}
\pgfpathmoveto{\pgfqpoint{0.117cm}{1.421cm}}
\pgfpathlineto{\pgfqpoint{0.682cm}{0.671cm}}
\pgfpathlineto{\pgfqpoint{1.246cm}{1.421cm}}
\pgfusepath{stroke}
\end{pgfscope}
\definecolor{eps2pgf_color}{gray}{0}\pgfsetstrokecolor{eps2pgf_color}\pgfsetfillcolor{eps2pgf_color}
\pgfpathmoveto{\pgfqpoint{0.273cm}{1.395cm}}
\pgfpathcurveto{\pgfqpoint{0.273cm}{1.432cm}}{\pgfqpoint{0.259cm}{1.467cm}}{\pgfqpoint{0.233cm}{1.492cm}}
\pgfpathcurveto{\pgfqpoint{0.207cm}{1.518cm}}{\pgfqpoint{0.173cm}{1.532cm}}{\pgfqpoint{0.137cm}{1.532cm}}
\pgfpathcurveto{\pgfqpoint{0.1cm}{1.532cm}}{\pgfqpoint{0.066cm}{1.518cm}}{\pgfqpoint{0.04cm}{1.492cm}}
\pgfpathcurveto{\pgfqpoint{0.014cm}{1.467cm}}{\pgfqpoint{0cm}{1.432cm}}{\pgfqpoint{0cm}{1.395cm}}
\pgfpathcurveto{\pgfqpoint{0cm}{1.359cm}}{\pgfqpoint{0.014cm}{1.324cm}}{\pgfqpoint{0.04cm}{1.299cm}}
\pgfpathcurveto{\pgfqpoint{0.066cm}{1.273cm}}{\pgfqpoint{0.1cm}{1.258cm}}{\pgfqpoint{0.137cm}{1.258cm}}
\pgfpathcurveto{\pgfqpoint{0.173cm}{1.258cm}}{\pgfqpoint{0.207cm}{1.273cm}}{\pgfqpoint{0.233cm}{1.299cm}}
\pgfpathcurveto{\pgfqpoint{0.259cm}{1.324cm}}{\pgfqpoint{0.273cm}{1.359cm}}{\pgfqpoint{0.273cm}{1.395cm}}
\pgfusepath{fill}
\begin{pgfscope}
\pgfsetdash{}{0cm}
\pgfsetlinewidth{0.818mm}
\pgfsetmiterlimit{7.0}
\pgfpathmoveto{\pgfqpoint{0.682cm}{0.671cm}}
\pgfpathlineto{\pgfqpoint{0.679cm}{1.418cm}}
\pgfusepath{stroke}
\end{pgfscope}
\pgfpathmoveto{\pgfqpoint{0.815cm}{1.399cm}}
\pgfpathcurveto{\pgfqpoint{0.815cm}{1.435cm}}{\pgfqpoint{0.801cm}{1.47cm}}{\pgfqpoint{0.775cm}{1.496cm}}
\pgfpathcurveto{\pgfqpoint{0.75cm}{1.521cm}}{\pgfqpoint{0.715cm}{1.536cm}}{\pgfqpoint{0.679cm}{1.536cm}}
\pgfpathcurveto{\pgfqpoint{0.643cm}{1.536cm}}{\pgfqpoint{0.608cm}{1.521cm}}{\pgfqpoint{0.582cm}{1.496cm}}
\pgfpathcurveto{\pgfqpoint{0.557cm}{1.47cm}}{\pgfqpoint{0.542cm}{1.435cm}}{\pgfqpoint{0.542cm}{1.399cm}}
\pgfpathcurveto{\pgfqpoint{0.542cm}{1.363cm}}{\pgfqpoint{0.557cm}{1.328cm}}{\pgfqpoint{0.582cm}{1.302cm}}
\pgfpathcurveto{\pgfqpoint{0.608cm}{1.276cm}}{\pgfqpoint{0.643cm}{1.262cm}}{\pgfqpoint{0.679cm}{1.262cm}}
\pgfpathcurveto{\pgfqpoint{0.715cm}{1.262cm}}{\pgfqpoint{0.75cm}{1.276cm}}{\pgfqpoint{0.775cm}{1.302cm}}
\pgfpathcurveto{\pgfqpoint{0.801cm}{1.328cm}}{\pgfqpoint{0.815cm}{1.363cm}}{\pgfqpoint{0.815cm}{1.399cm}}
\pgfusepath{fill}
\pgfpathmoveto{\pgfqpoint{1.345cm}{1.371cm}}
\pgfpathcurveto{\pgfqpoint{1.345cm}{1.408cm}}{\pgfqpoint{1.331cm}{1.442cm}}{\pgfqpoint{1.305cm}{1.468cm}}
\pgfpathcurveto{\pgfqpoint{1.28cm}{1.494cm}}{\pgfqpoint{1.245cm}{1.508cm}}{\pgfqpoint{1.209cm}{1.508cm}}
\pgfpathcurveto{\pgfqpoint{1.172cm}{1.508cm}}{\pgfqpoint{1.138cm}{1.494cm}}{\pgfqpoint{1.112cm}{1.468cm}}
\pgfpathcurveto{\pgfqpoint{1.087cm}{1.442cm}}{\pgfqpoint{1.072cm}{1.408cm}}{\pgfqpoint{1.072cm}{1.371cm}}
\pgfpathcurveto{\pgfqpoint{1.072cm}{1.335cm}}{\pgfqpoint{1.087cm}{1.3cm}}{\pgfqpoint{1.112cm}{1.274cm}}
\pgfpathcurveto{\pgfqpoint{1.138cm}{1.249cm}}{\pgfqpoint{1.172cm}{1.234cm}}{\pgfqpoint{1.209cm}{1.234cm}}
\pgfpathcurveto{\pgfqpoint{1.245cm}{1.234cm}}{\pgfqpoint{1.28cm}{1.249cm}}{\pgfqpoint{1.305cm}{1.274cm}}
\pgfpathcurveto{\pgfqpoint{1.331cm}{1.3cm}}{\pgfqpoint{1.345cm}{1.335cm}}{\pgfqpoint{1.345cm}{1.371cm}}
\pgfusepath{fill}
\begin{pgfscope}
\pgfsetdash{}{0cm}
\pgfsetlinewidth{0.818mm}
\pgfsetroundcap
\pgfsetmiterlimit{4.0}
\pgfpathmoveto{\pgfqpoint{0.682cm}{0.671cm}}
\pgfpathlineto{\pgfqpoint{0.682cm}{0.042cm}}
\pgfusepath{stroke}
\end{pgfscope}
\end{pgfscope}
\end{pgfscope}
\end{pgfscope}
\end{tikzpicture}}},X^{\!\resizebox{0.6em}{!}{
\begin{tikzpicture}
\pgfpathmoveto{\pgfqpoint{0cm}{-0.035cm}}
\pgfpathlineto{\pgfqpoint{1.376cm}{-0.035cm}}
\pgfpathlineto{\pgfqpoint{1.376cm}{1.552cm}}
\pgfpathlineto{\pgfqpoint{0cm}{1.552cm}}
\pgfpathclose
\pgfusepath{clip}
\begin{pgfscope}
\begin{pgfscope}
\pgfpathmoveto{\pgfqpoint{0cm}{-0.035cm}}
\pgfpathlineto{\pgfqpoint{1.376cm}{-0.035cm}}
\pgfpathlineto{\pgfqpoint{1.376cm}{1.552cm}}
\pgfpathlineto{\pgfqpoint{0cm}{1.552cm}}
\pgfpathclose
\pgfusepath{clip}
\begin{pgfscope}
\begin{pgfscope}
\pgfsetdash{}{0cm}
\pgfsetlinewidth{0.818mm}
\pgfsetroundcap
\pgfsetroundjoin
\pgfsetmiterlimit{7.0}
\definecolor{eps2pgf_color}{gray}{0}\pgfsetstrokecolor{eps2pgf_color}\pgfsetfillcolor{eps2pgf_color}
\pgfpathmoveto{\pgfqpoint{0.117cm}{1.421cm}}
\pgfpathlineto{\pgfqpoint{0.682cm}{0.671cm}}
\pgfpathlineto{\pgfqpoint{1.246cm}{1.421cm}}
\pgfusepath{stroke}
\end{pgfscope}
\definecolor{eps2pgf_color}{gray}{0}\pgfsetstrokecolor{eps2pgf_color}\pgfsetfillcolor{eps2pgf_color}
\pgfpathmoveto{\pgfqpoint{0.273cm}{1.395cm}}
\pgfpathcurveto{\pgfqpoint{0.273cm}{1.432cm}}{\pgfqpoint{0.259cm}{1.467cm}}{\pgfqpoint{0.233cm}{1.492cm}}
\pgfpathcurveto{\pgfqpoint{0.207cm}{1.518cm}}{\pgfqpoint{0.173cm}{1.532cm}}{\pgfqpoint{0.137cm}{1.532cm}}
\pgfpathcurveto{\pgfqpoint{0.1cm}{1.532cm}}{\pgfqpoint{0.066cm}{1.518cm}}{\pgfqpoint{0.04cm}{1.492cm}}
\pgfpathcurveto{\pgfqpoint{0.014cm}{1.467cm}}{\pgfqpoint{0cm}{1.432cm}}{\pgfqpoint{0cm}{1.395cm}}
\pgfpathcurveto{\pgfqpoint{0cm}{1.359cm}}{\pgfqpoint{0.014cm}{1.324cm}}{\pgfqpoint{0.04cm}{1.299cm}}
\pgfpathcurveto{\pgfqpoint{0.066cm}{1.273cm}}{\pgfqpoint{0.1cm}{1.258cm}}{\pgfqpoint{0.137cm}{1.258cm}}
\pgfpathcurveto{\pgfqpoint{0.173cm}{1.258cm}}{\pgfqpoint{0.207cm}{1.273cm}}{\pgfqpoint{0.233cm}{1.299cm}}
\pgfpathcurveto{\pgfqpoint{0.259cm}{1.324cm}}{\pgfqpoint{0.273cm}{1.359cm}}{\pgfqpoint{0.273cm}{1.395cm}}
\pgfusepath{fill}
\begin{pgfscope}
\pgfsetdash{}{0cm}
\pgfsetlinewidth{0.818mm}
\pgfsetmiterlimit{7.0}
\pgfpathmoveto{\pgfqpoint{0.682cm}{0.671cm}}
\pgfpathlineto{\pgfqpoint{0.679cm}{1.418cm}}
\pgfusepath{stroke}
\end{pgfscope}
\pgfpathmoveto{\pgfqpoint{0.815cm}{1.399cm}}
\pgfpathcurveto{\pgfqpoint{0.815cm}{1.435cm}}{\pgfqpoint{0.801cm}{1.47cm}}{\pgfqpoint{0.775cm}{1.496cm}}
\pgfpathcurveto{\pgfqpoint{0.75cm}{1.521cm}}{\pgfqpoint{0.715cm}{1.536cm}}{\pgfqpoint{0.679cm}{1.536cm}}
\pgfpathcurveto{\pgfqpoint{0.643cm}{1.536cm}}{\pgfqpoint{0.608cm}{1.521cm}}{\pgfqpoint{0.582cm}{1.496cm}}
\pgfpathcurveto{\pgfqpoint{0.557cm}{1.47cm}}{\pgfqpoint{0.542cm}{1.435cm}}{\pgfqpoint{0.542cm}{1.399cm}}
\pgfpathcurveto{\pgfqpoint{0.542cm}{1.363cm}}{\pgfqpoint{0.557cm}{1.328cm}}{\pgfqpoint{0.582cm}{1.302cm}}
\pgfpathcurveto{\pgfqpoint{0.608cm}{1.276cm}}{\pgfqpoint{0.643cm}{1.262cm}}{\pgfqpoint{0.679cm}{1.262cm}}
\pgfpathcurveto{\pgfqpoint{0.715cm}{1.262cm}}{\pgfqpoint{0.75cm}{1.276cm}}{\pgfqpoint{0.775cm}{1.302cm}}
\pgfpathcurveto{\pgfqpoint{0.801cm}{1.328cm}}{\pgfqpoint{0.815cm}{1.363cm}}{\pgfqpoint{0.815cm}{1.399cm}}
\pgfusepath{fill}
\pgfpathmoveto{\pgfqpoint{1.345cm}{1.371cm}}
\pgfpathcurveto{\pgfqpoint{1.345cm}{1.408cm}}{\pgfqpoint{1.331cm}{1.442cm}}{\pgfqpoint{1.305cm}{1.468cm}}
\pgfpathcurveto{\pgfqpoint{1.28cm}{1.494cm}}{\pgfqpoint{1.245cm}{1.508cm}}{\pgfqpoint{1.209cm}{1.508cm}}
\pgfpathcurveto{\pgfqpoint{1.172cm}{1.508cm}}{\pgfqpoint{1.138cm}{1.494cm}}{\pgfqpoint{1.112cm}{1.468cm}}
\pgfpathcurveto{\pgfqpoint{1.087cm}{1.442cm}}{\pgfqpoint{1.072cm}{1.408cm}}{\pgfqpoint{1.072cm}{1.371cm}}
\pgfpathcurveto{\pgfqpoint{1.072cm}{1.335cm}}{\pgfqpoint{1.087cm}{1.3cm}}{\pgfqpoint{1.112cm}{1.274cm}}
\pgfpathcurveto{\pgfqpoint{1.138cm}{1.249cm}}{\pgfqpoint{1.172cm}{1.234cm}}{\pgfqpoint{1.209cm}{1.234cm}}
\pgfpathcurveto{\pgfqpoint{1.245cm}{1.234cm}}{\pgfqpoint{1.28cm}{1.249cm}}{\pgfqpoint{1.305cm}{1.274cm}}
\pgfpathcurveto{\pgfqpoint{1.331cm}{1.3cm}}{\pgfqpoint{1.345cm}{1.335cm}}{\pgfqpoint{1.345cm}{1.371cm}}
\pgfusepath{fill}
\begin{pgfscope}
\pgfsetdash{}{0cm}
\pgfsetlinewidth{0.818mm}
\pgfsetroundcap
\pgfsetmiterlimit{4.0}
\pgfpathmoveto{\pgfqpoint{0.682cm}{0.671cm}}
\pgfpathlineto{\pgfqpoint{0.682cm}{0.042cm}}
\pgfusepath{stroke}
\end{pgfscope}
\end{pgfscope}
\end{pgfscope}
\end{pgfscope}
\end{tikzpicture}}},X)\|_{\CC^\gamma(\rho^{3+\gamma})}\lesssim 1,
\end{align*}
\begin{align*}
\|6X\circ(X^{\!\resizebox{0.6em}{!}{
\begin{tikzpicture}
\pgfpathmoveto{\pgfqpoint{0cm}{-0.035cm}}
\pgfpathlineto{\pgfqpoint{1.376cm}{-0.035cm}}
\pgfpathlineto{\pgfqpoint{1.376cm}{1.552cm}}
\pgfpathlineto{\pgfqpoint{0cm}{1.552cm}}
\pgfpathclose
\pgfusepath{clip}
\begin{pgfscope}
\begin{pgfscope}
\pgfpathmoveto{\pgfqpoint{0cm}{-0.035cm}}
\pgfpathlineto{\pgfqpoint{1.376cm}{-0.035cm}}
\pgfpathlineto{\pgfqpoint{1.376cm}{1.552cm}}
\pgfpathlineto{\pgfqpoint{0cm}{1.552cm}}
\pgfpathclose
\pgfusepath{clip}
\begin{pgfscope}
\begin{pgfscope}
\pgfsetdash{}{0cm}
\pgfsetlinewidth{0.818mm}
\pgfsetroundcap
\pgfsetroundjoin
\pgfsetmiterlimit{7.0}
\definecolor{eps2pgf_color}{gray}{0}\pgfsetstrokecolor{eps2pgf_color}\pgfsetfillcolor{eps2pgf_color}
\pgfpathmoveto{\pgfqpoint{0.117cm}{1.421cm}}
\pgfpathlineto{\pgfqpoint{0.682cm}{0.671cm}}
\pgfpathlineto{\pgfqpoint{1.246cm}{1.421cm}}
\pgfusepath{stroke}
\end{pgfscope}
\definecolor{eps2pgf_color}{gray}{0}\pgfsetstrokecolor{eps2pgf_color}\pgfsetfillcolor{eps2pgf_color}
\pgfpathmoveto{\pgfqpoint{0.273cm}{1.395cm}}
\pgfpathcurveto{\pgfqpoint{0.273cm}{1.432cm}}{\pgfqpoint{0.259cm}{1.467cm}}{\pgfqpoint{0.233cm}{1.492cm}}
\pgfpathcurveto{\pgfqpoint{0.207cm}{1.518cm}}{\pgfqpoint{0.173cm}{1.532cm}}{\pgfqpoint{0.137cm}{1.532cm}}
\pgfpathcurveto{\pgfqpoint{0.1cm}{1.532cm}}{\pgfqpoint{0.066cm}{1.518cm}}{\pgfqpoint{0.04cm}{1.492cm}}
\pgfpathcurveto{\pgfqpoint{0.014cm}{1.467cm}}{\pgfqpoint{0cm}{1.432cm}}{\pgfqpoint{0cm}{1.395cm}}
\pgfpathcurveto{\pgfqpoint{0cm}{1.359cm}}{\pgfqpoint{0.014cm}{1.324cm}}{\pgfqpoint{0.04cm}{1.299cm}}
\pgfpathcurveto{\pgfqpoint{0.066cm}{1.273cm}}{\pgfqpoint{0.1cm}{1.258cm}}{\pgfqpoint{0.137cm}{1.258cm}}
\pgfpathcurveto{\pgfqpoint{0.173cm}{1.258cm}}{\pgfqpoint{0.207cm}{1.273cm}}{\pgfqpoint{0.233cm}{1.299cm}}
\pgfpathcurveto{\pgfqpoint{0.259cm}{1.324cm}}{\pgfqpoint{0.273cm}{1.359cm}}{\pgfqpoint{0.273cm}{1.395cm}}
\pgfusepath{fill}
\begin{pgfscope}
\pgfsetdash{}{0cm}
\pgfsetlinewidth{0.818mm}
\pgfsetmiterlimit{7.0}
\pgfpathmoveto{\pgfqpoint{0.682cm}{0.671cm}}
\pgfpathlineto{\pgfqpoint{0.679cm}{1.418cm}}
\pgfusepath{stroke}
\end{pgfscope}
\pgfpathmoveto{\pgfqpoint{0.815cm}{1.399cm}}
\pgfpathcurveto{\pgfqpoint{0.815cm}{1.435cm}}{\pgfqpoint{0.801cm}{1.47cm}}{\pgfqpoint{0.775cm}{1.496cm}}
\pgfpathcurveto{\pgfqpoint{0.75cm}{1.521cm}}{\pgfqpoint{0.715cm}{1.536cm}}{\pgfqpoint{0.679cm}{1.536cm}}
\pgfpathcurveto{\pgfqpoint{0.643cm}{1.536cm}}{\pgfqpoint{0.608cm}{1.521cm}}{\pgfqpoint{0.582cm}{1.496cm}}
\pgfpathcurveto{\pgfqpoint{0.557cm}{1.47cm}}{\pgfqpoint{0.542cm}{1.435cm}}{\pgfqpoint{0.542cm}{1.399cm}}
\pgfpathcurveto{\pgfqpoint{0.542cm}{1.363cm}}{\pgfqpoint{0.557cm}{1.328cm}}{\pgfqpoint{0.582cm}{1.302cm}}
\pgfpathcurveto{\pgfqpoint{0.608cm}{1.276cm}}{\pgfqpoint{0.643cm}{1.262cm}}{\pgfqpoint{0.679cm}{1.262cm}}
\pgfpathcurveto{\pgfqpoint{0.715cm}{1.262cm}}{\pgfqpoint{0.75cm}{1.276cm}}{\pgfqpoint{0.775cm}{1.302cm}}
\pgfpathcurveto{\pgfqpoint{0.801cm}{1.328cm}}{\pgfqpoint{0.815cm}{1.363cm}}{\pgfqpoint{0.815cm}{1.399cm}}
\pgfusepath{fill}
\pgfpathmoveto{\pgfqpoint{1.345cm}{1.371cm}}
\pgfpathcurveto{\pgfqpoint{1.345cm}{1.408cm}}{\pgfqpoint{1.331cm}{1.442cm}}{\pgfqpoint{1.305cm}{1.468cm}}
\pgfpathcurveto{\pgfqpoint{1.28cm}{1.494cm}}{\pgfqpoint{1.245cm}{1.508cm}}{\pgfqpoint{1.209cm}{1.508cm}}
\pgfpathcurveto{\pgfqpoint{1.172cm}{1.508cm}}{\pgfqpoint{1.138cm}{1.494cm}}{\pgfqpoint{1.112cm}{1.468cm}}
\pgfpathcurveto{\pgfqpoint{1.087cm}{1.442cm}}{\pgfqpoint{1.072cm}{1.408cm}}{\pgfqpoint{1.072cm}{1.371cm}}
\pgfpathcurveto{\pgfqpoint{1.072cm}{1.335cm}}{\pgfqpoint{1.087cm}{1.3cm}}{\pgfqpoint{1.112cm}{1.274cm}}
\pgfpathcurveto{\pgfqpoint{1.138cm}{1.249cm}}{\pgfqpoint{1.172cm}{1.234cm}}{\pgfqpoint{1.209cm}{1.234cm}}
\pgfpathcurveto{\pgfqpoint{1.245cm}{1.234cm}}{\pgfqpoint{1.28cm}{1.249cm}}{\pgfqpoint{1.305cm}{1.274cm}}
\pgfpathcurveto{\pgfqpoint{1.331cm}{1.3cm}}{\pgfqpoint{1.345cm}{1.335cm}}{\pgfqpoint{1.345cm}{1.371cm}}
\pgfusepath{fill}
\begin{pgfscope}
\pgfsetdash{}{0cm}
\pgfsetlinewidth{0.818mm}
\pgfsetroundcap
\pgfsetmiterlimit{4.0}
\pgfpathmoveto{\pgfqpoint{0.682cm}{0.671cm}}
\pgfpathlineto{\pgfqpoint{0.682cm}{0.042cm}}
\pgfusepath{stroke}
\end{pgfscope}
\end{pgfscope}
\end{pgfscope}
\end{pgfscope}
\end{tikzpicture}}}\preccurlyeq(\phi+\psi))\|_{\CC^\gamma(\rho^{3+\gamma})}\lesssim\|\phi+\psi\|_{\CC^{\frac{1}{2}+\alpha}(\rho^{2+\alpha})}\lesssim  \|\psi\|_{L^\infty(\rho)}^\varepsilon+\|\psi\|_{\CC^{1+\alpha}(\rho^{2+\alpha})},
\end{align*}
\begin{align*}
\|6\mathrm{com}(\phi+\psi,X^{\!\resizebox{0.6em}{!}{
\begin{tikzpicture}
\pgfpathmoveto{\pgfqpoint{0cm}{-0.035cm}}
\pgfpathlineto{\pgfqpoint{1.376cm}{-0.035cm}}
\pgfpathlineto{\pgfqpoint{1.376cm}{1.552cm}}
\pgfpathlineto{\pgfqpoint{0cm}{1.552cm}}
\pgfpathclose
\pgfusepath{clip}
\begin{pgfscope}
\begin{pgfscope}
\pgfpathmoveto{\pgfqpoint{0cm}{-0.035cm}}
\pgfpathlineto{\pgfqpoint{1.376cm}{-0.035cm}}
\pgfpathlineto{\pgfqpoint{1.376cm}{1.552cm}}
\pgfpathlineto{\pgfqpoint{0cm}{1.552cm}}
\pgfpathclose
\pgfusepath{clip}
\begin{pgfscope}
\begin{pgfscope}
\pgfsetdash{}{0cm}
\pgfsetlinewidth{0.818mm}
\pgfsetroundcap
\pgfsetroundjoin
\pgfsetmiterlimit{7.0}
\definecolor{eps2pgf_color}{gray}{0}\pgfsetstrokecolor{eps2pgf_color}\pgfsetfillcolor{eps2pgf_color}
\pgfpathmoveto{\pgfqpoint{0.117cm}{1.421cm}}
\pgfpathlineto{\pgfqpoint{0.682cm}{0.671cm}}
\pgfpathlineto{\pgfqpoint{1.246cm}{1.421cm}}
\pgfusepath{stroke}
\end{pgfscope}
\definecolor{eps2pgf_color}{gray}{0}\pgfsetstrokecolor{eps2pgf_color}\pgfsetfillcolor{eps2pgf_color}
\pgfpathmoveto{\pgfqpoint{0.273cm}{1.395cm}}
\pgfpathcurveto{\pgfqpoint{0.273cm}{1.432cm}}{\pgfqpoint{0.259cm}{1.467cm}}{\pgfqpoint{0.233cm}{1.492cm}}
\pgfpathcurveto{\pgfqpoint{0.207cm}{1.518cm}}{\pgfqpoint{0.173cm}{1.532cm}}{\pgfqpoint{0.137cm}{1.532cm}}
\pgfpathcurveto{\pgfqpoint{0.1cm}{1.532cm}}{\pgfqpoint{0.066cm}{1.518cm}}{\pgfqpoint{0.04cm}{1.492cm}}
\pgfpathcurveto{\pgfqpoint{0.014cm}{1.467cm}}{\pgfqpoint{0cm}{1.432cm}}{\pgfqpoint{0cm}{1.395cm}}
\pgfpathcurveto{\pgfqpoint{0cm}{1.359cm}}{\pgfqpoint{0.014cm}{1.324cm}}{\pgfqpoint{0.04cm}{1.299cm}}
\pgfpathcurveto{\pgfqpoint{0.066cm}{1.273cm}}{\pgfqpoint{0.1cm}{1.258cm}}{\pgfqpoint{0.137cm}{1.258cm}}
\pgfpathcurveto{\pgfqpoint{0.173cm}{1.258cm}}{\pgfqpoint{0.207cm}{1.273cm}}{\pgfqpoint{0.233cm}{1.299cm}}
\pgfpathcurveto{\pgfqpoint{0.259cm}{1.324cm}}{\pgfqpoint{0.273cm}{1.359cm}}{\pgfqpoint{0.273cm}{1.395cm}}
\pgfusepath{fill}
\begin{pgfscope}
\pgfsetdash{}{0cm}
\pgfsetlinewidth{0.818mm}
\pgfsetmiterlimit{7.0}
\pgfpathmoveto{\pgfqpoint{0.682cm}{0.671cm}}
\pgfpathlineto{\pgfqpoint{0.679cm}{1.418cm}}
\pgfusepath{stroke}
\end{pgfscope}
\pgfpathmoveto{\pgfqpoint{0.815cm}{1.399cm}}
\pgfpathcurveto{\pgfqpoint{0.815cm}{1.435cm}}{\pgfqpoint{0.801cm}{1.47cm}}{\pgfqpoint{0.775cm}{1.496cm}}
\pgfpathcurveto{\pgfqpoint{0.75cm}{1.521cm}}{\pgfqpoint{0.715cm}{1.536cm}}{\pgfqpoint{0.679cm}{1.536cm}}
\pgfpathcurveto{\pgfqpoint{0.643cm}{1.536cm}}{\pgfqpoint{0.608cm}{1.521cm}}{\pgfqpoint{0.582cm}{1.496cm}}
\pgfpathcurveto{\pgfqpoint{0.557cm}{1.47cm}}{\pgfqpoint{0.542cm}{1.435cm}}{\pgfqpoint{0.542cm}{1.399cm}}
\pgfpathcurveto{\pgfqpoint{0.542cm}{1.363cm}}{\pgfqpoint{0.557cm}{1.328cm}}{\pgfqpoint{0.582cm}{1.302cm}}
\pgfpathcurveto{\pgfqpoint{0.608cm}{1.276cm}}{\pgfqpoint{0.643cm}{1.262cm}}{\pgfqpoint{0.679cm}{1.262cm}}
\pgfpathcurveto{\pgfqpoint{0.715cm}{1.262cm}}{\pgfqpoint{0.75cm}{1.276cm}}{\pgfqpoint{0.775cm}{1.302cm}}
\pgfpathcurveto{\pgfqpoint{0.801cm}{1.328cm}}{\pgfqpoint{0.815cm}{1.363cm}}{\pgfqpoint{0.815cm}{1.399cm}}
\pgfusepath{fill}
\pgfpathmoveto{\pgfqpoint{1.345cm}{1.371cm}}
\pgfpathcurveto{\pgfqpoint{1.345cm}{1.408cm}}{\pgfqpoint{1.331cm}{1.442cm}}{\pgfqpoint{1.305cm}{1.468cm}}
\pgfpathcurveto{\pgfqpoint{1.28cm}{1.494cm}}{\pgfqpoint{1.245cm}{1.508cm}}{\pgfqpoint{1.209cm}{1.508cm}}
\pgfpathcurveto{\pgfqpoint{1.172cm}{1.508cm}}{\pgfqpoint{1.138cm}{1.494cm}}{\pgfqpoint{1.112cm}{1.468cm}}
\pgfpathcurveto{\pgfqpoint{1.087cm}{1.442cm}}{\pgfqpoint{1.072cm}{1.408cm}}{\pgfqpoint{1.072cm}{1.371cm}}
\pgfpathcurveto{\pgfqpoint{1.072cm}{1.335cm}}{\pgfqpoint{1.087cm}{1.3cm}}{\pgfqpoint{1.112cm}{1.274cm}}
\pgfpathcurveto{\pgfqpoint{1.138cm}{1.249cm}}{\pgfqpoint{1.172cm}{1.234cm}}{\pgfqpoint{1.209cm}{1.234cm}}
\pgfpathcurveto{\pgfqpoint{1.245cm}{1.234cm}}{\pgfqpoint{1.28cm}{1.249cm}}{\pgfqpoint{1.305cm}{1.274cm}}
\pgfpathcurveto{\pgfqpoint{1.331cm}{1.3cm}}{\pgfqpoint{1.345cm}{1.335cm}}{\pgfqpoint{1.345cm}{1.371cm}}
\pgfusepath{fill}
\begin{pgfscope}
\pgfsetdash{}{0cm}
\pgfsetlinewidth{0.818mm}
\pgfsetroundcap
\pgfsetmiterlimit{4.0}
\pgfpathmoveto{\pgfqpoint{0.682cm}{0.671cm}}
\pgfpathlineto{\pgfqpoint{0.682cm}{0.042cm}}
\pgfusepath{stroke}
\end{pgfscope}
\end{pgfscope}
\end{pgfscope}
\end{pgfscope}
\end{tikzpicture}}},X)\|_{\CC^\gamma(\rho^{3+\gamma})}\lesssim\|\phi+\psi\|_{\CC^{\frac{1}{2}}(\rho^{2+\alpha})}\lesssim \|\psi\|_{L^\infty(\rho)}^\varepsilon+ \|\psi\|_{\CC^{1+\alpha}(\rho^{2+\alpha})},
\end{align*}
\begin{align*}
&\|(-X^{\!\resizebox{0.6em}{!}{
\begin{tikzpicture}
\pgfpathmoveto{\pgfqpoint{0cm}{-0.035cm}}
\pgfpathlineto{\pgfqpoint{1.376cm}{-0.035cm}}
\pgfpathlineto{\pgfqpoint{1.376cm}{1.552cm}}
\pgfpathlineto{\pgfqpoint{0cm}{1.552cm}}
\pgfpathclose
\pgfusepath{clip}
\begin{pgfscope}
\begin{pgfscope}
\pgfpathmoveto{\pgfqpoint{0cm}{-0.035cm}}
\pgfpathlineto{\pgfqpoint{1.376cm}{-0.035cm}}
\pgfpathlineto{\pgfqpoint{1.376cm}{1.552cm}}
\pgfpathlineto{\pgfqpoint{0cm}{1.552cm}}
\pgfpathclose
\pgfusepath{clip}
\begin{pgfscope}
\begin{pgfscope}
\pgfsetdash{}{0cm}
\pgfsetlinewidth{0.818mm}
\pgfsetroundcap
\pgfsetroundjoin
\pgfsetmiterlimit{7.0}
\definecolor{eps2pgf_color}{gray}{0}\pgfsetstrokecolor{eps2pgf_color}\pgfsetfillcolor{eps2pgf_color}
\pgfpathmoveto{\pgfqpoint{0.117cm}{1.421cm}}
\pgfpathlineto{\pgfqpoint{0.682cm}{0.671cm}}
\pgfpathlineto{\pgfqpoint{1.246cm}{1.421cm}}
\pgfusepath{stroke}
\end{pgfscope}
\definecolor{eps2pgf_color}{gray}{0}\pgfsetstrokecolor{eps2pgf_color}\pgfsetfillcolor{eps2pgf_color}
\pgfpathmoveto{\pgfqpoint{0.273cm}{1.395cm}}
\pgfpathcurveto{\pgfqpoint{0.273cm}{1.432cm}}{\pgfqpoint{0.259cm}{1.467cm}}{\pgfqpoint{0.233cm}{1.492cm}}
\pgfpathcurveto{\pgfqpoint{0.207cm}{1.518cm}}{\pgfqpoint{0.173cm}{1.532cm}}{\pgfqpoint{0.137cm}{1.532cm}}
\pgfpathcurveto{\pgfqpoint{0.1cm}{1.532cm}}{\pgfqpoint{0.066cm}{1.518cm}}{\pgfqpoint{0.04cm}{1.492cm}}
\pgfpathcurveto{\pgfqpoint{0.014cm}{1.467cm}}{\pgfqpoint{0cm}{1.432cm}}{\pgfqpoint{0cm}{1.395cm}}
\pgfpathcurveto{\pgfqpoint{0cm}{1.359cm}}{\pgfqpoint{0.014cm}{1.324cm}}{\pgfqpoint{0.04cm}{1.299cm}}
\pgfpathcurveto{\pgfqpoint{0.066cm}{1.273cm}}{\pgfqpoint{0.1cm}{1.258cm}}{\pgfqpoint{0.137cm}{1.258cm}}
\pgfpathcurveto{\pgfqpoint{0.173cm}{1.258cm}}{\pgfqpoint{0.207cm}{1.273cm}}{\pgfqpoint{0.233cm}{1.299cm}}
\pgfpathcurveto{\pgfqpoint{0.259cm}{1.324cm}}{\pgfqpoint{0.273cm}{1.359cm}}{\pgfqpoint{0.273cm}{1.395cm}}
\pgfusepath{fill}
\begin{pgfscope}
\pgfsetdash{}{0cm}
\pgfsetlinewidth{0.818mm}
\pgfsetmiterlimit{7.0}
\pgfpathmoveto{\pgfqpoint{0.682cm}{0.671cm}}
\pgfpathlineto{\pgfqpoint{0.679cm}{1.418cm}}
\pgfusepath{stroke}
\end{pgfscope}
\pgfpathmoveto{\pgfqpoint{0.815cm}{1.399cm}}
\pgfpathcurveto{\pgfqpoint{0.815cm}{1.435cm}}{\pgfqpoint{0.801cm}{1.47cm}}{\pgfqpoint{0.775cm}{1.496cm}}
\pgfpathcurveto{\pgfqpoint{0.75cm}{1.521cm}}{\pgfqpoint{0.715cm}{1.536cm}}{\pgfqpoint{0.679cm}{1.536cm}}
\pgfpathcurveto{\pgfqpoint{0.643cm}{1.536cm}}{\pgfqpoint{0.608cm}{1.521cm}}{\pgfqpoint{0.582cm}{1.496cm}}
\pgfpathcurveto{\pgfqpoint{0.557cm}{1.47cm}}{\pgfqpoint{0.542cm}{1.435cm}}{\pgfqpoint{0.542cm}{1.399cm}}
\pgfpathcurveto{\pgfqpoint{0.542cm}{1.363cm}}{\pgfqpoint{0.557cm}{1.328cm}}{\pgfqpoint{0.582cm}{1.302cm}}
\pgfpathcurveto{\pgfqpoint{0.608cm}{1.276cm}}{\pgfqpoint{0.643cm}{1.262cm}}{\pgfqpoint{0.679cm}{1.262cm}}
\pgfpathcurveto{\pgfqpoint{0.715cm}{1.262cm}}{\pgfqpoint{0.75cm}{1.276cm}}{\pgfqpoint{0.775cm}{1.302cm}}
\pgfpathcurveto{\pgfqpoint{0.801cm}{1.328cm}}{\pgfqpoint{0.815cm}{1.363cm}}{\pgfqpoint{0.815cm}{1.399cm}}
\pgfusepath{fill}
\pgfpathmoveto{\pgfqpoint{1.345cm}{1.371cm}}
\pgfpathcurveto{\pgfqpoint{1.345cm}{1.408cm}}{\pgfqpoint{1.331cm}{1.442cm}}{\pgfqpoint{1.305cm}{1.468cm}}
\pgfpathcurveto{\pgfqpoint{1.28cm}{1.494cm}}{\pgfqpoint{1.245cm}{1.508cm}}{\pgfqpoint{1.209cm}{1.508cm}}
\pgfpathcurveto{\pgfqpoint{1.172cm}{1.508cm}}{\pgfqpoint{1.138cm}{1.494cm}}{\pgfqpoint{1.112cm}{1.468cm}}
\pgfpathcurveto{\pgfqpoint{1.087cm}{1.442cm}}{\pgfqpoint{1.072cm}{1.408cm}}{\pgfqpoint{1.072cm}{1.371cm}}
\pgfpathcurveto{\pgfqpoint{1.072cm}{1.335cm}}{\pgfqpoint{1.087cm}{1.3cm}}{\pgfqpoint{1.112cm}{1.274cm}}
\pgfpathcurveto{\pgfqpoint{1.138cm}{1.249cm}}{\pgfqpoint{1.172cm}{1.234cm}}{\pgfqpoint{1.209cm}{1.234cm}}
\pgfpathcurveto{\pgfqpoint{1.245cm}{1.234cm}}{\pgfqpoint{1.28cm}{1.249cm}}{\pgfqpoint{1.305cm}{1.274cm}}
\pgfpathcurveto{\pgfqpoint{1.331cm}{1.3cm}}{\pgfqpoint{1.345cm}{1.335cm}}{\pgfqpoint{1.345cm}{1.371cm}}
\pgfusepath{fill}
\begin{pgfscope}
\pgfsetdash{}{0cm}
\pgfsetlinewidth{0.818mm}
\pgfsetroundcap
\pgfsetmiterlimit{4.0}
\pgfpathmoveto{\pgfqpoint{0.682cm}{0.671cm}}
\pgfpathlineto{\pgfqpoint{0.682cm}{0.042cm}}
\pgfusepath{stroke}
\end{pgfscope}
\end{pgfscope}
\end{pgfscope}
\end{pgfscope}
\end{tikzpicture}}}+\phi)^3+3(-X^{\!\resizebox{0.6em}{!}{
\begin{tikzpicture}
\pgfpathmoveto{\pgfqpoint{0cm}{-0.035cm}}
\pgfpathlineto{\pgfqpoint{1.376cm}{-0.035cm}}
\pgfpathlineto{\pgfqpoint{1.376cm}{1.552cm}}
\pgfpathlineto{\pgfqpoint{0cm}{1.552cm}}
\pgfpathclose
\pgfusepath{clip}
\begin{pgfscope}
\begin{pgfscope}
\pgfpathmoveto{\pgfqpoint{0cm}{-0.035cm}}
\pgfpathlineto{\pgfqpoint{1.376cm}{-0.035cm}}
\pgfpathlineto{\pgfqpoint{1.376cm}{1.552cm}}
\pgfpathlineto{\pgfqpoint{0cm}{1.552cm}}
\pgfpathclose
\pgfusepath{clip}
\begin{pgfscope}
\begin{pgfscope}
\pgfsetdash{}{0cm}
\pgfsetlinewidth{0.818mm}
\pgfsetroundcap
\pgfsetroundjoin
\pgfsetmiterlimit{7.0}
\definecolor{eps2pgf_color}{gray}{0}\pgfsetstrokecolor{eps2pgf_color}\pgfsetfillcolor{eps2pgf_color}
\pgfpathmoveto{\pgfqpoint{0.117cm}{1.421cm}}
\pgfpathlineto{\pgfqpoint{0.682cm}{0.671cm}}
\pgfpathlineto{\pgfqpoint{1.246cm}{1.421cm}}
\pgfusepath{stroke}
\end{pgfscope}
\definecolor{eps2pgf_color}{gray}{0}\pgfsetstrokecolor{eps2pgf_color}\pgfsetfillcolor{eps2pgf_color}
\pgfpathmoveto{\pgfqpoint{0.273cm}{1.395cm}}
\pgfpathcurveto{\pgfqpoint{0.273cm}{1.432cm}}{\pgfqpoint{0.259cm}{1.467cm}}{\pgfqpoint{0.233cm}{1.492cm}}
\pgfpathcurveto{\pgfqpoint{0.207cm}{1.518cm}}{\pgfqpoint{0.173cm}{1.532cm}}{\pgfqpoint{0.137cm}{1.532cm}}
\pgfpathcurveto{\pgfqpoint{0.1cm}{1.532cm}}{\pgfqpoint{0.066cm}{1.518cm}}{\pgfqpoint{0.04cm}{1.492cm}}
\pgfpathcurveto{\pgfqpoint{0.014cm}{1.467cm}}{\pgfqpoint{0cm}{1.432cm}}{\pgfqpoint{0cm}{1.395cm}}
\pgfpathcurveto{\pgfqpoint{0cm}{1.359cm}}{\pgfqpoint{0.014cm}{1.324cm}}{\pgfqpoint{0.04cm}{1.299cm}}
\pgfpathcurveto{\pgfqpoint{0.066cm}{1.273cm}}{\pgfqpoint{0.1cm}{1.258cm}}{\pgfqpoint{0.137cm}{1.258cm}}
\pgfpathcurveto{\pgfqpoint{0.173cm}{1.258cm}}{\pgfqpoint{0.207cm}{1.273cm}}{\pgfqpoint{0.233cm}{1.299cm}}
\pgfpathcurveto{\pgfqpoint{0.259cm}{1.324cm}}{\pgfqpoint{0.273cm}{1.359cm}}{\pgfqpoint{0.273cm}{1.395cm}}
\pgfusepath{fill}
\begin{pgfscope}
\pgfsetdash{}{0cm}
\pgfsetlinewidth{0.818mm}
\pgfsetmiterlimit{7.0}
\pgfpathmoveto{\pgfqpoint{0.682cm}{0.671cm}}
\pgfpathlineto{\pgfqpoint{0.679cm}{1.418cm}}
\pgfusepath{stroke}
\end{pgfscope}
\pgfpathmoveto{\pgfqpoint{0.815cm}{1.399cm}}
\pgfpathcurveto{\pgfqpoint{0.815cm}{1.435cm}}{\pgfqpoint{0.801cm}{1.47cm}}{\pgfqpoint{0.775cm}{1.496cm}}
\pgfpathcurveto{\pgfqpoint{0.75cm}{1.521cm}}{\pgfqpoint{0.715cm}{1.536cm}}{\pgfqpoint{0.679cm}{1.536cm}}
\pgfpathcurveto{\pgfqpoint{0.643cm}{1.536cm}}{\pgfqpoint{0.608cm}{1.521cm}}{\pgfqpoint{0.582cm}{1.496cm}}
\pgfpathcurveto{\pgfqpoint{0.557cm}{1.47cm}}{\pgfqpoint{0.542cm}{1.435cm}}{\pgfqpoint{0.542cm}{1.399cm}}
\pgfpathcurveto{\pgfqpoint{0.542cm}{1.363cm}}{\pgfqpoint{0.557cm}{1.328cm}}{\pgfqpoint{0.582cm}{1.302cm}}
\pgfpathcurveto{\pgfqpoint{0.608cm}{1.276cm}}{\pgfqpoint{0.643cm}{1.262cm}}{\pgfqpoint{0.679cm}{1.262cm}}
\pgfpathcurveto{\pgfqpoint{0.715cm}{1.262cm}}{\pgfqpoint{0.75cm}{1.276cm}}{\pgfqpoint{0.775cm}{1.302cm}}
\pgfpathcurveto{\pgfqpoint{0.801cm}{1.328cm}}{\pgfqpoint{0.815cm}{1.363cm}}{\pgfqpoint{0.815cm}{1.399cm}}
\pgfusepath{fill}
\pgfpathmoveto{\pgfqpoint{1.345cm}{1.371cm}}
\pgfpathcurveto{\pgfqpoint{1.345cm}{1.408cm}}{\pgfqpoint{1.331cm}{1.442cm}}{\pgfqpoint{1.305cm}{1.468cm}}
\pgfpathcurveto{\pgfqpoint{1.28cm}{1.494cm}}{\pgfqpoint{1.245cm}{1.508cm}}{\pgfqpoint{1.209cm}{1.508cm}}
\pgfpathcurveto{\pgfqpoint{1.172cm}{1.508cm}}{\pgfqpoint{1.138cm}{1.494cm}}{\pgfqpoint{1.112cm}{1.468cm}}
\pgfpathcurveto{\pgfqpoint{1.087cm}{1.442cm}}{\pgfqpoint{1.072cm}{1.408cm}}{\pgfqpoint{1.072cm}{1.371cm}}
\pgfpathcurveto{\pgfqpoint{1.072cm}{1.335cm}}{\pgfqpoint{1.087cm}{1.3cm}}{\pgfqpoint{1.112cm}{1.274cm}}
\pgfpathcurveto{\pgfqpoint{1.138cm}{1.249cm}}{\pgfqpoint{1.172cm}{1.234cm}}{\pgfqpoint{1.209cm}{1.234cm}}
\pgfpathcurveto{\pgfqpoint{1.245cm}{1.234cm}}{\pgfqpoint{1.28cm}{1.249cm}}{\pgfqpoint{1.305cm}{1.274cm}}
\pgfpathcurveto{\pgfqpoint{1.331cm}{1.3cm}}{\pgfqpoint{1.345cm}{1.335cm}}{\pgfqpoint{1.345cm}{1.371cm}}
\pgfusepath{fill}
\begin{pgfscope}
\pgfsetdash{}{0cm}
\pgfsetlinewidth{0.818mm}
\pgfsetroundcap
\pgfsetmiterlimit{4.0}
\pgfpathmoveto{\pgfqpoint{0.682cm}{0.671cm}}
\pgfpathlineto{\pgfqpoint{0.682cm}{0.042cm}}
\pgfusepath{stroke}
\end{pgfscope}
\end{pgfscope}
\end{pgfscope}
\end{pgfscope}
\end{tikzpicture}}}+\phi)^2\psi+3(-X^{\!\resizebox{0.6em}{!}{
\begin{tikzpicture}
\pgfpathmoveto{\pgfqpoint{0cm}{-0.035cm}}
\pgfpathlineto{\pgfqpoint{1.376cm}{-0.035cm}}
\pgfpathlineto{\pgfqpoint{1.376cm}{1.552cm}}
\pgfpathlineto{\pgfqpoint{0cm}{1.552cm}}
\pgfpathclose
\pgfusepath{clip}
\begin{pgfscope}
\begin{pgfscope}
\pgfpathmoveto{\pgfqpoint{0cm}{-0.035cm}}
\pgfpathlineto{\pgfqpoint{1.376cm}{-0.035cm}}
\pgfpathlineto{\pgfqpoint{1.376cm}{1.552cm}}
\pgfpathlineto{\pgfqpoint{0cm}{1.552cm}}
\pgfpathclose
\pgfusepath{clip}
\begin{pgfscope}
\begin{pgfscope}
\pgfsetdash{}{0cm}
\pgfsetlinewidth{0.818mm}
\pgfsetroundcap
\pgfsetroundjoin
\pgfsetmiterlimit{7.0}
\definecolor{eps2pgf_color}{gray}{0}\pgfsetstrokecolor{eps2pgf_color}\pgfsetfillcolor{eps2pgf_color}
\pgfpathmoveto{\pgfqpoint{0.117cm}{1.421cm}}
\pgfpathlineto{\pgfqpoint{0.682cm}{0.671cm}}
\pgfpathlineto{\pgfqpoint{1.246cm}{1.421cm}}
\pgfusepath{stroke}
\end{pgfscope}
\definecolor{eps2pgf_color}{gray}{0}\pgfsetstrokecolor{eps2pgf_color}\pgfsetfillcolor{eps2pgf_color}
\pgfpathmoveto{\pgfqpoint{0.273cm}{1.395cm}}
\pgfpathcurveto{\pgfqpoint{0.273cm}{1.432cm}}{\pgfqpoint{0.259cm}{1.467cm}}{\pgfqpoint{0.233cm}{1.492cm}}
\pgfpathcurveto{\pgfqpoint{0.207cm}{1.518cm}}{\pgfqpoint{0.173cm}{1.532cm}}{\pgfqpoint{0.137cm}{1.532cm}}
\pgfpathcurveto{\pgfqpoint{0.1cm}{1.532cm}}{\pgfqpoint{0.066cm}{1.518cm}}{\pgfqpoint{0.04cm}{1.492cm}}
\pgfpathcurveto{\pgfqpoint{0.014cm}{1.467cm}}{\pgfqpoint{0cm}{1.432cm}}{\pgfqpoint{0cm}{1.395cm}}
\pgfpathcurveto{\pgfqpoint{0cm}{1.359cm}}{\pgfqpoint{0.014cm}{1.324cm}}{\pgfqpoint{0.04cm}{1.299cm}}
\pgfpathcurveto{\pgfqpoint{0.066cm}{1.273cm}}{\pgfqpoint{0.1cm}{1.258cm}}{\pgfqpoint{0.137cm}{1.258cm}}
\pgfpathcurveto{\pgfqpoint{0.173cm}{1.258cm}}{\pgfqpoint{0.207cm}{1.273cm}}{\pgfqpoint{0.233cm}{1.299cm}}
\pgfpathcurveto{\pgfqpoint{0.259cm}{1.324cm}}{\pgfqpoint{0.273cm}{1.359cm}}{\pgfqpoint{0.273cm}{1.395cm}}
\pgfusepath{fill}
\begin{pgfscope}
\pgfsetdash{}{0cm}
\pgfsetlinewidth{0.818mm}
\pgfsetmiterlimit{7.0}
\pgfpathmoveto{\pgfqpoint{0.682cm}{0.671cm}}
\pgfpathlineto{\pgfqpoint{0.679cm}{1.418cm}}
\pgfusepath{stroke}
\end{pgfscope}
\pgfpathmoveto{\pgfqpoint{0.815cm}{1.399cm}}
\pgfpathcurveto{\pgfqpoint{0.815cm}{1.435cm}}{\pgfqpoint{0.801cm}{1.47cm}}{\pgfqpoint{0.775cm}{1.496cm}}
\pgfpathcurveto{\pgfqpoint{0.75cm}{1.521cm}}{\pgfqpoint{0.715cm}{1.536cm}}{\pgfqpoint{0.679cm}{1.536cm}}
\pgfpathcurveto{\pgfqpoint{0.643cm}{1.536cm}}{\pgfqpoint{0.608cm}{1.521cm}}{\pgfqpoint{0.582cm}{1.496cm}}
\pgfpathcurveto{\pgfqpoint{0.557cm}{1.47cm}}{\pgfqpoint{0.542cm}{1.435cm}}{\pgfqpoint{0.542cm}{1.399cm}}
\pgfpathcurveto{\pgfqpoint{0.542cm}{1.363cm}}{\pgfqpoint{0.557cm}{1.328cm}}{\pgfqpoint{0.582cm}{1.302cm}}
\pgfpathcurveto{\pgfqpoint{0.608cm}{1.276cm}}{\pgfqpoint{0.643cm}{1.262cm}}{\pgfqpoint{0.679cm}{1.262cm}}
\pgfpathcurveto{\pgfqpoint{0.715cm}{1.262cm}}{\pgfqpoint{0.75cm}{1.276cm}}{\pgfqpoint{0.775cm}{1.302cm}}
\pgfpathcurveto{\pgfqpoint{0.801cm}{1.328cm}}{\pgfqpoint{0.815cm}{1.363cm}}{\pgfqpoint{0.815cm}{1.399cm}}
\pgfusepath{fill}
\pgfpathmoveto{\pgfqpoint{1.345cm}{1.371cm}}
\pgfpathcurveto{\pgfqpoint{1.345cm}{1.408cm}}{\pgfqpoint{1.331cm}{1.442cm}}{\pgfqpoint{1.305cm}{1.468cm}}
\pgfpathcurveto{\pgfqpoint{1.28cm}{1.494cm}}{\pgfqpoint{1.245cm}{1.508cm}}{\pgfqpoint{1.209cm}{1.508cm}}
\pgfpathcurveto{\pgfqpoint{1.172cm}{1.508cm}}{\pgfqpoint{1.138cm}{1.494cm}}{\pgfqpoint{1.112cm}{1.468cm}}
\pgfpathcurveto{\pgfqpoint{1.087cm}{1.442cm}}{\pgfqpoint{1.072cm}{1.408cm}}{\pgfqpoint{1.072cm}{1.371cm}}
\pgfpathcurveto{\pgfqpoint{1.072cm}{1.335cm}}{\pgfqpoint{1.087cm}{1.3cm}}{\pgfqpoint{1.112cm}{1.274cm}}
\pgfpathcurveto{\pgfqpoint{1.138cm}{1.249cm}}{\pgfqpoint{1.172cm}{1.234cm}}{\pgfqpoint{1.209cm}{1.234cm}}
\pgfpathcurveto{\pgfqpoint{1.245cm}{1.234cm}}{\pgfqpoint{1.28cm}{1.249cm}}{\pgfqpoint{1.305cm}{1.274cm}}
\pgfpathcurveto{\pgfqpoint{1.331cm}{1.3cm}}{\pgfqpoint{1.345cm}{1.335cm}}{\pgfqpoint{1.345cm}{1.371cm}}
\pgfusepath{fill}
\begin{pgfscope}
\pgfsetdash{}{0cm}
\pgfsetlinewidth{0.818mm}
\pgfsetroundcap
\pgfsetmiterlimit{4.0}
\pgfpathmoveto{\pgfqpoint{0.682cm}{0.671cm}}
\pgfpathlineto{\pgfqpoint{0.682cm}{0.042cm}}
\pgfusepath{stroke}
\end{pgfscope}
\end{pgfscope}
\end{pgfscope}
\end{pgfscope}
\end{tikzpicture}}}+\phi)\psi^2\|_{\CC^\gamma(\rho^{3+\gamma})}\\
&\quad\lesssim (1+\|\psi\|_{L^\infty(\rho)})(1+\|\psi\|_{\CC^\gamma(\rho^{1+\gamma})}).
\end{align*}
Next, we observe that due to our choice of $K$ at the end of Section \ref{s:phi}, it holds that $$2^{(1+\gamma+\kappa)K}\simeq 1+\|\psi\|_{L^\infty(\rho)}^{1+\varepsilon},\quad 2^{(\gamma+\kappa)K/2}\simeq 1+\|\psi\|_{L^\infty(\rho)}^\varepsilon,$$
$$
2^{(\frac{1}{2}+\gamma+\kappa)\frac{2}{3}K}\simeq 1+\|\psi\|_{L^\infty(\rho)}^{\varepsilon}, \quad 2^{(\frac{1}{2}+\gamma+\kappa)\frac{4}{3}K}\simeq 1+\|\psi\|_{L^\infty(\rho)}^{\varepsilon}
$$
 for some $\varepsilon\in(0,1)$ (whose value possibly changes from bound to bound); and in view of Table~\ref{t:loc} we have
 $$
\|\UU_\leqslant\llbracket X^2 \rrbracket\|_{\CC^\gamma(\rho^{2})}\lesssim 2^{(1+\gamma+\kappa)K}\|\llbracket X^2 \rrbracket\|_{\CC^{-1-\kappa}(\rho^{\sigma})},
$$
\begin{align*}
\|\UU_\leqslant X^{\!\resizebox{!}{.8em}{
\begin{tikzpicture}
\pgfpathmoveto{\pgfqpoint{0cm}{-0.035cm}}
\pgfpathlineto{\pgfqpoint{1.976cm}{-0.035cm}}
\pgfpathlineto{\pgfqpoint{1.976cm}{1.94cm}}
\pgfpathlineto{\pgfqpoint{0cm}{1.94cm}}
\pgfpathclose
\pgfusepath{clip}
\begin{pgfscope}
\begin{pgfscope}
\pgfpathmoveto{\pgfqpoint{0cm}{-0.035cm}}
\pgfpathlineto{\pgfqpoint{1.976cm}{-0.035cm}}
\pgfpathlineto{\pgfqpoint{1.976cm}{1.94cm}}
\pgfpathlineto{\pgfqpoint{0cm}{1.94cm}}
\pgfpathclose
\pgfusepath{clip}
\begin{pgfscope}
\begin{pgfscope}
\pgfsetdash{}{0cm}
\pgfsetlinewidth{0.818mm}
\pgfsetroundcap
\pgfsetroundjoin
\pgfsetmiterlimit{7.0}
\definecolor{eps2pgf_color}{gray}{0}\pgfsetstrokecolor{eps2pgf_color}\pgfsetfillcolor{eps2pgf_color}
\pgfpathmoveto{\pgfqpoint{0.117cm}{1.815cm}}
\pgfpathlineto{\pgfqpoint{0.682cm}{1.065cm}}
\pgfpathlineto{\pgfqpoint{1.246cm}{1.815cm}}
\pgfusepath{stroke}
\end{pgfscope}
\definecolor{eps2pgf_color}{gray}{0}\pgfsetstrokecolor{eps2pgf_color}\pgfsetfillcolor{eps2pgf_color}
\pgfpathmoveto{\pgfqpoint{0.273cm}{1.789cm}}
\pgfpathcurveto{\pgfqpoint{0.273cm}{1.825cm}}{\pgfqpoint{0.259cm}{1.86cm}}{\pgfqpoint{0.233cm}{1.886cm}}
\pgfpathcurveto{\pgfqpoint{0.207cm}{1.912cm}}{\pgfqpoint{0.173cm}{1.926cm}}{\pgfqpoint{0.137cm}{1.926cm}}
\pgfpathcurveto{\pgfqpoint{0.1cm}{1.926cm}}{\pgfqpoint{0.066cm}{1.912cm}}{\pgfqpoint{0.04cm}{1.886cm}}
\pgfpathcurveto{\pgfqpoint{0.014cm}{1.86cm}}{\pgfqpoint{0cm}{1.825cm}}{\pgfqpoint{0cm}{1.789cm}}
\pgfpathcurveto{\pgfqpoint{0cm}{1.753cm}}{\pgfqpoint{0.014cm}{1.718cm}}{\pgfqpoint{0.04cm}{1.692cm}}
\pgfpathcurveto{\pgfqpoint{0.066cm}{1.667cm}}{\pgfqpoint{0.1cm}{1.652cm}}{\pgfqpoint{0.137cm}{1.652cm}}
\pgfpathcurveto{\pgfqpoint{0.173cm}{1.652cm}}{\pgfqpoint{0.207cm}{1.667cm}}{\pgfqpoint{0.233cm}{1.692cm}}
\pgfpathcurveto{\pgfqpoint{0.259cm}{1.718cm}}{\pgfqpoint{0.273cm}{1.753cm}}{\pgfqpoint{0.273cm}{1.789cm}}
\pgfusepath{fill}
\pgfpathmoveto{\pgfqpoint{1.345cm}{1.765cm}}
\pgfpathcurveto{\pgfqpoint{1.345cm}{1.801cm}}{\pgfqpoint{1.331cm}{1.836cm}}{\pgfqpoint{1.305cm}{1.862cm}}
\pgfpathcurveto{\pgfqpoint{1.28cm}{1.887cm}}{\pgfqpoint{1.245cm}{1.902cm}}{\pgfqpoint{1.209cm}{1.902cm}}
\pgfpathcurveto{\pgfqpoint{1.172cm}{1.902cm}}{\pgfqpoint{1.138cm}{1.887cm}}{\pgfqpoint{1.112cm}{1.862cm}}
\pgfpathcurveto{\pgfqpoint{1.087cm}{1.836cm}}{\pgfqpoint{1.072cm}{1.801cm}}{\pgfqpoint{1.072cm}{1.765cm}}
\pgfpathcurveto{\pgfqpoint{1.072cm}{1.728cm}}{\pgfqpoint{1.087cm}{1.694cm}}{\pgfqpoint{1.112cm}{1.668cm}}
\pgfpathcurveto{\pgfqpoint{1.138cm}{1.642cm}}{\pgfqpoint{1.172cm}{1.628cm}}{\pgfqpoint{1.209cm}{1.628cm}}
\pgfpathcurveto{\pgfqpoint{1.245cm}{1.628cm}}{\pgfqpoint{1.28cm}{1.642cm}}{\pgfqpoint{1.305cm}{1.668cm}}
\pgfpathcurveto{\pgfqpoint{1.331cm}{1.694cm}}{\pgfqpoint{1.345cm}{1.728cm}}{\pgfqpoint{1.345cm}{1.765cm}}
\pgfusepath{fill}
\begin{pgfscope}
\pgfsetdash{}{0cm}
\pgfsetlinewidth{0.818mm}
\pgfsetroundcap
\pgfsetroundjoin
\pgfsetmiterlimit{7.0}
\pgfpathmoveto{\pgfqpoint{0.682cm}{1.065cm}}
\pgfpathlineto{\pgfqpoint{1.246cm}{0.315cm}}
\pgfpathlineto{\pgfqpoint{1.811cm}{1.065cm}}
\pgfusepath{stroke}
\end{pgfscope}
\pgfpathmoveto{\pgfqpoint{1.948cm}{1.065cm}}
\pgfpathcurveto{\pgfqpoint{1.948cm}{1.101cm}}{\pgfqpoint{1.933cm}{1.136cm}}{\pgfqpoint{1.907cm}{1.162cm}}
\pgfpathcurveto{\pgfqpoint{1.882cm}{1.187cm}}{\pgfqpoint{1.847cm}{1.202cm}}{\pgfqpoint{1.811cm}{1.202cm}}
\pgfpathcurveto{\pgfqpoint{1.775cm}{1.202cm}}{\pgfqpoint{1.74cm}{1.187cm}}{\pgfqpoint{1.714cm}{1.162cm}}
\pgfpathcurveto{\pgfqpoint{1.689cm}{1.136cm}}{\pgfqpoint{1.674cm}{1.101cm}}{\pgfqpoint{1.674cm}{1.065cm}}
\pgfpathcurveto{\pgfqpoint{1.674cm}{1.029cm}}{\pgfqpoint{1.689cm}{0.994cm}}{\pgfqpoint{1.714cm}{0.968cm}}
\pgfpathcurveto{\pgfqpoint{1.74cm}{0.942cm}}{\pgfqpoint{1.775cm}{0.928cm}}{\pgfqpoint{1.811cm}{0.928cm}}
\pgfpathcurveto{\pgfqpoint{1.847cm}{0.928cm}}{\pgfqpoint{1.882cm}{0.942cm}}{\pgfqpoint{1.907cm}{0.968cm}}
\pgfpathcurveto{\pgfqpoint{1.933cm}{0.994cm}}{\pgfqpoint{1.948cm}{1.029cm}}{\pgfqpoint{1.948cm}{1.065cm}}
\pgfusepath{fill}
\begin{pgfscope}
\pgfsetdash{}{0cm}
\pgfsetlinewidth{0.818mm}
\pgfsetmiterlimit{7.0}
\pgfpathmoveto{\pgfqpoint{1.246cm}{0.315cm}}
\pgfpathlineto{\pgfqpoint{1.244cm}{1.061cm}}
\pgfusepath{stroke}
\end{pgfscope}
\pgfpathmoveto{\pgfqpoint{1.38cm}{1.065cm}}
\pgfpathcurveto{\pgfqpoint{1.38cm}{1.101cm}}{\pgfqpoint{1.366cm}{1.136cm}}{\pgfqpoint{1.34cm}{1.162cm}}
\pgfpathcurveto{\pgfqpoint{1.315cm}{1.187cm}}{\pgfqpoint{1.28cm}{1.202cm}}{\pgfqpoint{1.244cm}{1.202cm}}
\pgfpathcurveto{\pgfqpoint{1.207cm}{1.202cm}}{\pgfqpoint{1.173cm}{1.187cm}}{\pgfqpoint{1.147cm}{1.162cm}}
\pgfpathcurveto{\pgfqpoint{1.121cm}{1.136cm}}{\pgfqpoint{1.107cm}{1.101cm}}{\pgfqpoint{1.107cm}{1.065cm}}
\pgfpathcurveto{\pgfqpoint{1.107cm}{1.029cm}}{\pgfqpoint{1.121cm}{0.994cm}}{\pgfqpoint{1.147cm}{0.968cm}}
\pgfpathcurveto{\pgfqpoint{1.173cm}{0.942cm}}{\pgfqpoint{1.207cm}{0.928cm}}{\pgfqpoint{1.244cm}{0.928cm}}
\pgfpathcurveto{\pgfqpoint{1.28cm}{0.928cm}}{\pgfqpoint{1.315cm}{0.942cm}}{\pgfqpoint{1.34cm}{0.968cm}}
\pgfpathcurveto{\pgfqpoint{1.366cm}{0.994cm}}{\pgfqpoint{1.38cm}{1.029cm}}{\pgfqpoint{1.38cm}{1.065cm}}
\pgfusepath{fill}
\begin{pgfscope}
\pgfsetdash{}{0cm}
\pgfsetlinewidth{0.818mm}
\pgfsetmiterlimit{4.0}
\pgfpathmoveto{\pgfqpoint{1.383cm}{0.178cm}}
\pgfpathcurveto{\pgfqpoint{1.383cm}{0.214cm}}{\pgfqpoint{1.369cm}{0.249cm}}{\pgfqpoint{1.343cm}{0.275cm}}
\pgfpathcurveto{\pgfqpoint{1.317cm}{0.3cm}}{\pgfqpoint{1.283cm}{0.315cm}}{\pgfqpoint{1.246cm}{0.315cm}}
\pgfpathcurveto{\pgfqpoint{1.21cm}{0.315cm}}{\pgfqpoint{1.175cm}{0.3cm}}{\pgfqpoint{1.15cm}{0.275cm}}
\pgfpathcurveto{\pgfqpoint{1.124cm}{0.249cm}}{\pgfqpoint{1.11cm}{0.214cm}}{\pgfqpoint{1.11cm}{0.178cm}}
\pgfpathcurveto{\pgfqpoint{1.11cm}{0.141cm}}{\pgfqpoint{1.124cm}{0.107cm}}{\pgfqpoint{1.15cm}{0.081cm}}
\pgfpathcurveto{\pgfqpoint{1.175cm}{0.055cm}}{\pgfqpoint{1.21cm}{0.041cm}}{\pgfqpoint{1.246cm}{0.041cm}}
\pgfpathcurveto{\pgfqpoint{1.283cm}{0.041cm}}{\pgfqpoint{1.317cm}{0.055cm}}{\pgfqpoint{1.343cm}{0.081cm}}
\pgfpathcurveto{\pgfqpoint{1.369cm}{0.107cm}}{\pgfqpoint{1.383cm}{0.141cm}}{\pgfqpoint{1.383cm}{0.178cm}}
\pgfusepath{stroke}
\end{pgfscope}
\end{pgfscope}
\end{pgfscope}
\end{pgfscope}
\end{tikzpicture}}}\|_{\CC^{\gamma}(\rho^{2+\gamma})}\lesssim 2^{(\gamma+\kappa)K/2}\|X^{\!\resizebox{!}{.8em}{
\begin{tikzpicture}
\pgfpathmoveto{\pgfqpoint{0cm}{-0.035cm}}
\pgfpathlineto{\pgfqpoint{1.976cm}{-0.035cm}}
\pgfpathlineto{\pgfqpoint{1.976cm}{1.94cm}}
\pgfpathlineto{\pgfqpoint{0cm}{1.94cm}}
\pgfpathclose
\pgfusepath{clip}
\begin{pgfscope}
\begin{pgfscope}
\pgfpathmoveto{\pgfqpoint{0cm}{-0.035cm}}
\pgfpathlineto{\pgfqpoint{1.976cm}{-0.035cm}}
\pgfpathlineto{\pgfqpoint{1.976cm}{1.94cm}}
\pgfpathlineto{\pgfqpoint{0cm}{1.94cm}}
\pgfpathclose
\pgfusepath{clip}
\begin{pgfscope}
\begin{pgfscope}
\pgfsetdash{}{0cm}
\pgfsetlinewidth{0.818mm}
\pgfsetroundcap
\pgfsetroundjoin
\pgfsetmiterlimit{7.0}
\definecolor{eps2pgf_color}{gray}{0}\pgfsetstrokecolor{eps2pgf_color}\pgfsetfillcolor{eps2pgf_color}
\pgfpathmoveto{\pgfqpoint{0.117cm}{1.815cm}}
\pgfpathlineto{\pgfqpoint{0.682cm}{1.065cm}}
\pgfpathlineto{\pgfqpoint{1.246cm}{1.815cm}}
\pgfusepath{stroke}
\end{pgfscope}
\definecolor{eps2pgf_color}{gray}{0}\pgfsetstrokecolor{eps2pgf_color}\pgfsetfillcolor{eps2pgf_color}
\pgfpathmoveto{\pgfqpoint{0.273cm}{1.789cm}}
\pgfpathcurveto{\pgfqpoint{0.273cm}{1.825cm}}{\pgfqpoint{0.259cm}{1.86cm}}{\pgfqpoint{0.233cm}{1.886cm}}
\pgfpathcurveto{\pgfqpoint{0.207cm}{1.912cm}}{\pgfqpoint{0.173cm}{1.926cm}}{\pgfqpoint{0.137cm}{1.926cm}}
\pgfpathcurveto{\pgfqpoint{0.1cm}{1.926cm}}{\pgfqpoint{0.066cm}{1.912cm}}{\pgfqpoint{0.04cm}{1.886cm}}
\pgfpathcurveto{\pgfqpoint{0.014cm}{1.86cm}}{\pgfqpoint{0cm}{1.825cm}}{\pgfqpoint{0cm}{1.789cm}}
\pgfpathcurveto{\pgfqpoint{0cm}{1.753cm}}{\pgfqpoint{0.014cm}{1.718cm}}{\pgfqpoint{0.04cm}{1.692cm}}
\pgfpathcurveto{\pgfqpoint{0.066cm}{1.667cm}}{\pgfqpoint{0.1cm}{1.652cm}}{\pgfqpoint{0.137cm}{1.652cm}}
\pgfpathcurveto{\pgfqpoint{0.173cm}{1.652cm}}{\pgfqpoint{0.207cm}{1.667cm}}{\pgfqpoint{0.233cm}{1.692cm}}
\pgfpathcurveto{\pgfqpoint{0.259cm}{1.718cm}}{\pgfqpoint{0.273cm}{1.753cm}}{\pgfqpoint{0.273cm}{1.789cm}}
\pgfusepath{fill}
\pgfpathmoveto{\pgfqpoint{1.345cm}{1.765cm}}
\pgfpathcurveto{\pgfqpoint{1.345cm}{1.801cm}}{\pgfqpoint{1.331cm}{1.836cm}}{\pgfqpoint{1.305cm}{1.862cm}}
\pgfpathcurveto{\pgfqpoint{1.28cm}{1.887cm}}{\pgfqpoint{1.245cm}{1.902cm}}{\pgfqpoint{1.209cm}{1.902cm}}
\pgfpathcurveto{\pgfqpoint{1.172cm}{1.902cm}}{\pgfqpoint{1.138cm}{1.887cm}}{\pgfqpoint{1.112cm}{1.862cm}}
\pgfpathcurveto{\pgfqpoint{1.087cm}{1.836cm}}{\pgfqpoint{1.072cm}{1.801cm}}{\pgfqpoint{1.072cm}{1.765cm}}
\pgfpathcurveto{\pgfqpoint{1.072cm}{1.728cm}}{\pgfqpoint{1.087cm}{1.694cm}}{\pgfqpoint{1.112cm}{1.668cm}}
\pgfpathcurveto{\pgfqpoint{1.138cm}{1.642cm}}{\pgfqpoint{1.172cm}{1.628cm}}{\pgfqpoint{1.209cm}{1.628cm}}
\pgfpathcurveto{\pgfqpoint{1.245cm}{1.628cm}}{\pgfqpoint{1.28cm}{1.642cm}}{\pgfqpoint{1.305cm}{1.668cm}}
\pgfpathcurveto{\pgfqpoint{1.331cm}{1.694cm}}{\pgfqpoint{1.345cm}{1.728cm}}{\pgfqpoint{1.345cm}{1.765cm}}
\pgfusepath{fill}
\begin{pgfscope}
\pgfsetdash{}{0cm}
\pgfsetlinewidth{0.818mm}
\pgfsetroundcap
\pgfsetroundjoin
\pgfsetmiterlimit{7.0}
\pgfpathmoveto{\pgfqpoint{0.682cm}{1.065cm}}
\pgfpathlineto{\pgfqpoint{1.246cm}{0.315cm}}
\pgfpathlineto{\pgfqpoint{1.811cm}{1.065cm}}
\pgfusepath{stroke}
\end{pgfscope}
\pgfpathmoveto{\pgfqpoint{1.948cm}{1.065cm}}
\pgfpathcurveto{\pgfqpoint{1.948cm}{1.101cm}}{\pgfqpoint{1.933cm}{1.136cm}}{\pgfqpoint{1.907cm}{1.162cm}}
\pgfpathcurveto{\pgfqpoint{1.882cm}{1.187cm}}{\pgfqpoint{1.847cm}{1.202cm}}{\pgfqpoint{1.811cm}{1.202cm}}
\pgfpathcurveto{\pgfqpoint{1.775cm}{1.202cm}}{\pgfqpoint{1.74cm}{1.187cm}}{\pgfqpoint{1.714cm}{1.162cm}}
\pgfpathcurveto{\pgfqpoint{1.689cm}{1.136cm}}{\pgfqpoint{1.674cm}{1.101cm}}{\pgfqpoint{1.674cm}{1.065cm}}
\pgfpathcurveto{\pgfqpoint{1.674cm}{1.029cm}}{\pgfqpoint{1.689cm}{0.994cm}}{\pgfqpoint{1.714cm}{0.968cm}}
\pgfpathcurveto{\pgfqpoint{1.74cm}{0.942cm}}{\pgfqpoint{1.775cm}{0.928cm}}{\pgfqpoint{1.811cm}{0.928cm}}
\pgfpathcurveto{\pgfqpoint{1.847cm}{0.928cm}}{\pgfqpoint{1.882cm}{0.942cm}}{\pgfqpoint{1.907cm}{0.968cm}}
\pgfpathcurveto{\pgfqpoint{1.933cm}{0.994cm}}{\pgfqpoint{1.948cm}{1.029cm}}{\pgfqpoint{1.948cm}{1.065cm}}
\pgfusepath{fill}
\begin{pgfscope}
\pgfsetdash{}{0cm}
\pgfsetlinewidth{0.818mm}
\pgfsetmiterlimit{7.0}
\pgfpathmoveto{\pgfqpoint{1.246cm}{0.315cm}}
\pgfpathlineto{\pgfqpoint{1.244cm}{1.061cm}}
\pgfusepath{stroke}
\end{pgfscope}
\pgfpathmoveto{\pgfqpoint{1.38cm}{1.065cm}}
\pgfpathcurveto{\pgfqpoint{1.38cm}{1.101cm}}{\pgfqpoint{1.366cm}{1.136cm}}{\pgfqpoint{1.34cm}{1.162cm}}
\pgfpathcurveto{\pgfqpoint{1.315cm}{1.187cm}}{\pgfqpoint{1.28cm}{1.202cm}}{\pgfqpoint{1.244cm}{1.202cm}}
\pgfpathcurveto{\pgfqpoint{1.207cm}{1.202cm}}{\pgfqpoint{1.173cm}{1.187cm}}{\pgfqpoint{1.147cm}{1.162cm}}
\pgfpathcurveto{\pgfqpoint{1.121cm}{1.136cm}}{\pgfqpoint{1.107cm}{1.101cm}}{\pgfqpoint{1.107cm}{1.065cm}}
\pgfpathcurveto{\pgfqpoint{1.107cm}{1.029cm}}{\pgfqpoint{1.121cm}{0.994cm}}{\pgfqpoint{1.147cm}{0.968cm}}
\pgfpathcurveto{\pgfqpoint{1.173cm}{0.942cm}}{\pgfqpoint{1.207cm}{0.928cm}}{\pgfqpoint{1.244cm}{0.928cm}}
\pgfpathcurveto{\pgfqpoint{1.28cm}{0.928cm}}{\pgfqpoint{1.315cm}{0.942cm}}{\pgfqpoint{1.34cm}{0.968cm}}
\pgfpathcurveto{\pgfqpoint{1.366cm}{0.994cm}}{\pgfqpoint{1.38cm}{1.029cm}}{\pgfqpoint{1.38cm}{1.065cm}}
\pgfusepath{fill}
\begin{pgfscope}
\pgfsetdash{}{0cm}
\pgfsetlinewidth{0.818mm}
\pgfsetmiterlimit{4.0}
\pgfpathmoveto{\pgfqpoint{1.383cm}{0.178cm}}
\pgfpathcurveto{\pgfqpoint{1.383cm}{0.214cm}}{\pgfqpoint{1.369cm}{0.249cm}}{\pgfqpoint{1.343cm}{0.275cm}}
\pgfpathcurveto{\pgfqpoint{1.317cm}{0.3cm}}{\pgfqpoint{1.283cm}{0.315cm}}{\pgfqpoint{1.246cm}{0.315cm}}
\pgfpathcurveto{\pgfqpoint{1.21cm}{0.315cm}}{\pgfqpoint{1.175cm}{0.3cm}}{\pgfqpoint{1.15cm}{0.275cm}}
\pgfpathcurveto{\pgfqpoint{1.124cm}{0.249cm}}{\pgfqpoint{1.11cm}{0.214cm}}{\pgfqpoint{1.11cm}{0.178cm}}
\pgfpathcurveto{\pgfqpoint{1.11cm}{0.141cm}}{\pgfqpoint{1.124cm}{0.107cm}}{\pgfqpoint{1.15cm}{0.081cm}}
\pgfpathcurveto{\pgfqpoint{1.175cm}{0.055cm}}{\pgfqpoint{1.21cm}{0.041cm}}{\pgfqpoint{1.246cm}{0.041cm}}
\pgfpathcurveto{\pgfqpoint{1.283cm}{0.041cm}}{\pgfqpoint{1.317cm}{0.055cm}}{\pgfqpoint{1.343cm}{0.081cm}}
\pgfpathcurveto{\pgfqpoint{1.369cm}{0.107cm}}{\pgfqpoint{1.383cm}{0.141cm}}{\pgfqpoint{1.383cm}{0.178cm}}
\pgfusepath{stroke}
\end{pgfscope}
\end{pgfscope}
\end{pgfscope}
\end{pgfscope}
\end{tikzpicture}}}\|_{\CC^{-\kappa}(\rho^\sigma)},
\end{align*}
\begin{align*}
\|\UU_\leqslant X\|_{\CC^\gamma(\rho^{2})}\lesssim 2^{(\frac{1}{2}+\gamma+\kappa)\frac{2}{3}K} \|X\|_{\CC^{-\frac{1}{2}-\kappa}(\rho^\sigma)},
\end{align*}
$$
\|\UU_{\leq}X\|_{\CC^{-\frac{1}{2}-\kappa}(\rho^{\sigma})}\lesssim \|X\|_{\CC^{-\frac{1}{2}-\kappa}(\rho^{\sigma})},
$$
\begin{align*}
\|\UU_\leqslant X^{\!\resizebox{!}{.8em}{
\begin{tikzpicture}
\pgfpathmoveto{\pgfqpoint{0cm}{-0.035cm}}
\pgfpathlineto{\pgfqpoint{1.976cm}{-0.035cm}}
\pgfpathlineto{\pgfqpoint{1.976cm}{1.94cm}}
\pgfpathlineto{\pgfqpoint{0cm}{1.94cm}}
\pgfpathclose
\pgfusepath{clip}
\begin{pgfscope}
\begin{pgfscope}
\pgfpathmoveto{\pgfqpoint{0cm}{-0.035cm}}
\pgfpathlineto{\pgfqpoint{1.976cm}{-0.035cm}}
\pgfpathlineto{\pgfqpoint{1.976cm}{1.94cm}}
\pgfpathlineto{\pgfqpoint{0cm}{1.94cm}}
\pgfpathclose
\pgfusepath{clip}
\begin{pgfscope}
\begin{pgfscope}
\pgfsetdash{}{0cm}
\pgfsetlinewidth{0.818mm}
\pgfsetroundcap
\pgfsetroundjoin
\pgfsetmiterlimit{7.0}
\definecolor{eps2pgf_color}{gray}{0}\pgfsetstrokecolor{eps2pgf_color}\pgfsetfillcolor{eps2pgf_color}
\pgfpathmoveto{\pgfqpoint{0.117cm}{1.815cm}}
\pgfpathlineto{\pgfqpoint{0.682cm}{1.065cm}}
\pgfpathlineto{\pgfqpoint{1.246cm}{1.815cm}}
\pgfusepath{stroke}
\end{pgfscope}
\definecolor{eps2pgf_color}{gray}{0}\pgfsetstrokecolor{eps2pgf_color}\pgfsetfillcolor{eps2pgf_color}
\pgfpathmoveto{\pgfqpoint{0.273cm}{1.789cm}}
\pgfpathcurveto{\pgfqpoint{0.273cm}{1.825cm}}{\pgfqpoint{0.259cm}{1.86cm}}{\pgfqpoint{0.233cm}{1.886cm}}
\pgfpathcurveto{\pgfqpoint{0.207cm}{1.912cm}}{\pgfqpoint{0.173cm}{1.926cm}}{\pgfqpoint{0.137cm}{1.926cm}}
\pgfpathcurveto{\pgfqpoint{0.1cm}{1.926cm}}{\pgfqpoint{0.066cm}{1.912cm}}{\pgfqpoint{0.04cm}{1.886cm}}
\pgfpathcurveto{\pgfqpoint{0.014cm}{1.86cm}}{\pgfqpoint{0cm}{1.825cm}}{\pgfqpoint{0cm}{1.789cm}}
\pgfpathcurveto{\pgfqpoint{0cm}{1.753cm}}{\pgfqpoint{0.014cm}{1.718cm}}{\pgfqpoint{0.04cm}{1.692cm}}
\pgfpathcurveto{\pgfqpoint{0.066cm}{1.667cm}}{\pgfqpoint{0.1cm}{1.652cm}}{\pgfqpoint{0.137cm}{1.652cm}}
\pgfpathcurveto{\pgfqpoint{0.173cm}{1.652cm}}{\pgfqpoint{0.207cm}{1.667cm}}{\pgfqpoint{0.233cm}{1.692cm}}
\pgfpathcurveto{\pgfqpoint{0.259cm}{1.718cm}}{\pgfqpoint{0.273cm}{1.753cm}}{\pgfqpoint{0.273cm}{1.789cm}}
\pgfusepath{fill}
\begin{pgfscope}
\pgfsetdash{}{0cm}
\pgfsetlinewidth{0.818mm}
\pgfsetmiterlimit{7.0}
\pgfpathmoveto{\pgfqpoint{0.682cm}{1.065cm}}
\pgfpathlineto{\pgfqpoint{0.679cm}{1.812cm}}
\pgfusepath{stroke}
\end{pgfscope}
\pgfpathmoveto{\pgfqpoint{0.815cm}{1.793cm}}
\pgfpathcurveto{\pgfqpoint{0.815cm}{1.829cm}}{\pgfqpoint{0.801cm}{1.864cm}}{\pgfqpoint{0.775cm}{1.89cm}}
\pgfpathcurveto{\pgfqpoint{0.75cm}{1.915cm}}{\pgfqpoint{0.715cm}{1.93cm}}{\pgfqpoint{0.679cm}{1.93cm}}
\pgfpathcurveto{\pgfqpoint{0.643cm}{1.93cm}}{\pgfqpoint{0.608cm}{1.915cm}}{\pgfqpoint{0.582cm}{1.89cm}}
\pgfpathcurveto{\pgfqpoint{0.557cm}{1.864cm}}{\pgfqpoint{0.542cm}{1.829cm}}{\pgfqpoint{0.542cm}{1.793cm}}
\pgfpathcurveto{\pgfqpoint{0.542cm}{1.756cm}}{\pgfqpoint{0.557cm}{1.722cm}}{\pgfqpoint{0.582cm}{1.696cm}}
\pgfpathcurveto{\pgfqpoint{0.608cm}{1.67cm}}{\pgfqpoint{0.643cm}{1.656cm}}{\pgfqpoint{0.679cm}{1.656cm}}
\pgfpathcurveto{\pgfqpoint{0.715cm}{1.656cm}}{\pgfqpoint{0.75cm}{1.67cm}}{\pgfqpoint{0.775cm}{1.696cm}}
\pgfpathcurveto{\pgfqpoint{0.801cm}{1.722cm}}{\pgfqpoint{0.815cm}{1.756cm}}{\pgfqpoint{0.815cm}{1.793cm}}
\pgfusepath{fill}
\pgfpathmoveto{\pgfqpoint{1.345cm}{1.765cm}}
\pgfpathcurveto{\pgfqpoint{1.345cm}{1.801cm}}{\pgfqpoint{1.331cm}{1.836cm}}{\pgfqpoint{1.305cm}{1.862cm}}
\pgfpathcurveto{\pgfqpoint{1.28cm}{1.887cm}}{\pgfqpoint{1.245cm}{1.902cm}}{\pgfqpoint{1.209cm}{1.902cm}}
\pgfpathcurveto{\pgfqpoint{1.172cm}{1.902cm}}{\pgfqpoint{1.138cm}{1.887cm}}{\pgfqpoint{1.112cm}{1.862cm}}
\pgfpathcurveto{\pgfqpoint{1.087cm}{1.836cm}}{\pgfqpoint{1.072cm}{1.801cm}}{\pgfqpoint{1.072cm}{1.765cm}}
\pgfpathcurveto{\pgfqpoint{1.072cm}{1.728cm}}{\pgfqpoint{1.087cm}{1.694cm}}{\pgfqpoint{1.112cm}{1.668cm}}
\pgfpathcurveto{\pgfqpoint{1.138cm}{1.642cm}}{\pgfqpoint{1.172cm}{1.628cm}}{\pgfqpoint{1.209cm}{1.628cm}}
\pgfpathcurveto{\pgfqpoint{1.245cm}{1.628cm}}{\pgfqpoint{1.28cm}{1.642cm}}{\pgfqpoint{1.305cm}{1.668cm}}
\pgfpathcurveto{\pgfqpoint{1.331cm}{1.694cm}}{\pgfqpoint{1.345cm}{1.728cm}}{\pgfqpoint{1.345cm}{1.765cm}}
\pgfusepath{fill}
\begin{pgfscope}
\pgfsetdash{}{0cm}
\pgfsetlinewidth{0.818mm}
\pgfsetroundcap
\pgfsetroundjoin
\pgfsetmiterlimit{7.0}
\pgfpathmoveto{\pgfqpoint{0.682cm}{1.065cm}}
\pgfpathlineto{\pgfqpoint{1.246cm}{0.315cm}}
\pgfpathlineto{\pgfqpoint{1.811cm}{1.065cm}}
\pgfusepath{stroke}
\end{pgfscope}
\pgfpathmoveto{\pgfqpoint{1.948cm}{1.065cm}}
\pgfpathcurveto{\pgfqpoint{1.948cm}{1.101cm}}{\pgfqpoint{1.933cm}{1.136cm}}{\pgfqpoint{1.907cm}{1.162cm}}
\pgfpathcurveto{\pgfqpoint{1.882cm}{1.187cm}}{\pgfqpoint{1.847cm}{1.202cm}}{\pgfqpoint{1.811cm}{1.202cm}}
\pgfpathcurveto{\pgfqpoint{1.775cm}{1.202cm}}{\pgfqpoint{1.74cm}{1.187cm}}{\pgfqpoint{1.714cm}{1.162cm}}
\pgfpathcurveto{\pgfqpoint{1.689cm}{1.136cm}}{\pgfqpoint{1.674cm}{1.101cm}}{\pgfqpoint{1.674cm}{1.065cm}}
\pgfpathcurveto{\pgfqpoint{1.674cm}{1.029cm}}{\pgfqpoint{1.689cm}{0.994cm}}{\pgfqpoint{1.714cm}{0.968cm}}
\pgfpathcurveto{\pgfqpoint{1.74cm}{0.942cm}}{\pgfqpoint{1.775cm}{0.928cm}}{\pgfqpoint{1.811cm}{0.928cm}}
\pgfpathcurveto{\pgfqpoint{1.847cm}{0.928cm}}{\pgfqpoint{1.882cm}{0.942cm}}{\pgfqpoint{1.907cm}{0.968cm}}
\pgfpathcurveto{\pgfqpoint{1.933cm}{0.994cm}}{\pgfqpoint{1.948cm}{1.029cm}}{\pgfqpoint{1.948cm}{1.065cm}}
\pgfusepath{fill}
\begin{pgfscope}
\pgfsetdash{}{0cm}
\pgfsetlinewidth{0.818mm}
\pgfsetmiterlimit{4.0}
\pgfpathmoveto{\pgfqpoint{1.383cm}{0.178cm}}
\pgfpathcurveto{\pgfqpoint{1.383cm}{0.214cm}}{\pgfqpoint{1.369cm}{0.249cm}}{\pgfqpoint{1.343cm}{0.275cm}}
\pgfpathcurveto{\pgfqpoint{1.317cm}{0.3cm}}{\pgfqpoint{1.283cm}{0.315cm}}{\pgfqpoint{1.246cm}{0.315cm}}
\pgfpathcurveto{\pgfqpoint{1.21cm}{0.315cm}}{\pgfqpoint{1.175cm}{0.3cm}}{\pgfqpoint{1.15cm}{0.275cm}}
\pgfpathcurveto{\pgfqpoint{1.124cm}{0.249cm}}{\pgfqpoint{1.11cm}{0.214cm}}{\pgfqpoint{1.11cm}{0.178cm}}
\pgfpathcurveto{\pgfqpoint{1.11cm}{0.141cm}}{\pgfqpoint{1.124cm}{0.107cm}}{\pgfqpoint{1.15cm}{0.081cm}}
\pgfpathcurveto{\pgfqpoint{1.175cm}{0.055cm}}{\pgfqpoint{1.21cm}{0.041cm}}{\pgfqpoint{1.246cm}{0.041cm}}
\pgfpathcurveto{\pgfqpoint{1.283cm}{0.041cm}}{\pgfqpoint{1.317cm}{0.055cm}}{\pgfqpoint{1.343cm}{0.081cm}}
\pgfpathcurveto{\pgfqpoint{1.369cm}{0.107cm}}{\pgfqpoint{1.383cm}{0.141cm}}{\pgfqpoint{1.383cm}{0.178cm}}
\pgfusepath{stroke}
\end{pgfscope}
\end{pgfscope}
\end{pgfscope}
\end{pgfscope}
\end{tikzpicture}}}\|_{\CC^\gamma(\rho^{2+\gamma})}\lesssim 2^{(\gamma+\kappa)K/2} \| X^{\!\resizebox{!}{.8em}{
\begin{tikzpicture}
\pgfpathmoveto{\pgfqpoint{0cm}{-0.035cm}}
\pgfpathlineto{\pgfqpoint{1.976cm}{-0.035cm}}
\pgfpathlineto{\pgfqpoint{1.976cm}{1.94cm}}
\pgfpathlineto{\pgfqpoint{0cm}{1.94cm}}
\pgfpathclose
\pgfusepath{clip}
\begin{pgfscope}
\begin{pgfscope}
\pgfpathmoveto{\pgfqpoint{0cm}{-0.035cm}}
\pgfpathlineto{\pgfqpoint{1.976cm}{-0.035cm}}
\pgfpathlineto{\pgfqpoint{1.976cm}{1.94cm}}
\pgfpathlineto{\pgfqpoint{0cm}{1.94cm}}
\pgfpathclose
\pgfusepath{clip}
\begin{pgfscope}
\begin{pgfscope}
\pgfsetdash{}{0cm}
\pgfsetlinewidth{0.818mm}
\pgfsetroundcap
\pgfsetroundjoin
\pgfsetmiterlimit{7.0}
\definecolor{eps2pgf_color}{gray}{0}\pgfsetstrokecolor{eps2pgf_color}\pgfsetfillcolor{eps2pgf_color}
\pgfpathmoveto{\pgfqpoint{0.117cm}{1.815cm}}
\pgfpathlineto{\pgfqpoint{0.682cm}{1.065cm}}
\pgfpathlineto{\pgfqpoint{1.246cm}{1.815cm}}
\pgfusepath{stroke}
\end{pgfscope}
\definecolor{eps2pgf_color}{gray}{0}\pgfsetstrokecolor{eps2pgf_color}\pgfsetfillcolor{eps2pgf_color}
\pgfpathmoveto{\pgfqpoint{0.273cm}{1.789cm}}
\pgfpathcurveto{\pgfqpoint{0.273cm}{1.825cm}}{\pgfqpoint{0.259cm}{1.86cm}}{\pgfqpoint{0.233cm}{1.886cm}}
\pgfpathcurveto{\pgfqpoint{0.207cm}{1.912cm}}{\pgfqpoint{0.173cm}{1.926cm}}{\pgfqpoint{0.137cm}{1.926cm}}
\pgfpathcurveto{\pgfqpoint{0.1cm}{1.926cm}}{\pgfqpoint{0.066cm}{1.912cm}}{\pgfqpoint{0.04cm}{1.886cm}}
\pgfpathcurveto{\pgfqpoint{0.014cm}{1.86cm}}{\pgfqpoint{0cm}{1.825cm}}{\pgfqpoint{0cm}{1.789cm}}
\pgfpathcurveto{\pgfqpoint{0cm}{1.753cm}}{\pgfqpoint{0.014cm}{1.718cm}}{\pgfqpoint{0.04cm}{1.692cm}}
\pgfpathcurveto{\pgfqpoint{0.066cm}{1.667cm}}{\pgfqpoint{0.1cm}{1.652cm}}{\pgfqpoint{0.137cm}{1.652cm}}
\pgfpathcurveto{\pgfqpoint{0.173cm}{1.652cm}}{\pgfqpoint{0.207cm}{1.667cm}}{\pgfqpoint{0.233cm}{1.692cm}}
\pgfpathcurveto{\pgfqpoint{0.259cm}{1.718cm}}{\pgfqpoint{0.273cm}{1.753cm}}{\pgfqpoint{0.273cm}{1.789cm}}
\pgfusepath{fill}
\begin{pgfscope}
\pgfsetdash{}{0cm}
\pgfsetlinewidth{0.818mm}
\pgfsetmiterlimit{7.0}
\pgfpathmoveto{\pgfqpoint{0.682cm}{1.065cm}}
\pgfpathlineto{\pgfqpoint{0.679cm}{1.812cm}}
\pgfusepath{stroke}
\end{pgfscope}
\pgfpathmoveto{\pgfqpoint{0.815cm}{1.793cm}}
\pgfpathcurveto{\pgfqpoint{0.815cm}{1.829cm}}{\pgfqpoint{0.801cm}{1.864cm}}{\pgfqpoint{0.775cm}{1.89cm}}
\pgfpathcurveto{\pgfqpoint{0.75cm}{1.915cm}}{\pgfqpoint{0.715cm}{1.93cm}}{\pgfqpoint{0.679cm}{1.93cm}}
\pgfpathcurveto{\pgfqpoint{0.643cm}{1.93cm}}{\pgfqpoint{0.608cm}{1.915cm}}{\pgfqpoint{0.582cm}{1.89cm}}
\pgfpathcurveto{\pgfqpoint{0.557cm}{1.864cm}}{\pgfqpoint{0.542cm}{1.829cm}}{\pgfqpoint{0.542cm}{1.793cm}}
\pgfpathcurveto{\pgfqpoint{0.542cm}{1.756cm}}{\pgfqpoint{0.557cm}{1.722cm}}{\pgfqpoint{0.582cm}{1.696cm}}
\pgfpathcurveto{\pgfqpoint{0.608cm}{1.67cm}}{\pgfqpoint{0.643cm}{1.656cm}}{\pgfqpoint{0.679cm}{1.656cm}}
\pgfpathcurveto{\pgfqpoint{0.715cm}{1.656cm}}{\pgfqpoint{0.75cm}{1.67cm}}{\pgfqpoint{0.775cm}{1.696cm}}
\pgfpathcurveto{\pgfqpoint{0.801cm}{1.722cm}}{\pgfqpoint{0.815cm}{1.756cm}}{\pgfqpoint{0.815cm}{1.793cm}}
\pgfusepath{fill}
\pgfpathmoveto{\pgfqpoint{1.345cm}{1.765cm}}
\pgfpathcurveto{\pgfqpoint{1.345cm}{1.801cm}}{\pgfqpoint{1.331cm}{1.836cm}}{\pgfqpoint{1.305cm}{1.862cm}}
\pgfpathcurveto{\pgfqpoint{1.28cm}{1.887cm}}{\pgfqpoint{1.245cm}{1.902cm}}{\pgfqpoint{1.209cm}{1.902cm}}
\pgfpathcurveto{\pgfqpoint{1.172cm}{1.902cm}}{\pgfqpoint{1.138cm}{1.887cm}}{\pgfqpoint{1.112cm}{1.862cm}}
\pgfpathcurveto{\pgfqpoint{1.087cm}{1.836cm}}{\pgfqpoint{1.072cm}{1.801cm}}{\pgfqpoint{1.072cm}{1.765cm}}
\pgfpathcurveto{\pgfqpoint{1.072cm}{1.728cm}}{\pgfqpoint{1.087cm}{1.694cm}}{\pgfqpoint{1.112cm}{1.668cm}}
\pgfpathcurveto{\pgfqpoint{1.138cm}{1.642cm}}{\pgfqpoint{1.172cm}{1.628cm}}{\pgfqpoint{1.209cm}{1.628cm}}
\pgfpathcurveto{\pgfqpoint{1.245cm}{1.628cm}}{\pgfqpoint{1.28cm}{1.642cm}}{\pgfqpoint{1.305cm}{1.668cm}}
\pgfpathcurveto{\pgfqpoint{1.331cm}{1.694cm}}{\pgfqpoint{1.345cm}{1.728cm}}{\pgfqpoint{1.345cm}{1.765cm}}
\pgfusepath{fill}
\begin{pgfscope}
\pgfsetdash{}{0cm}
\pgfsetlinewidth{0.818mm}
\pgfsetroundcap
\pgfsetroundjoin
\pgfsetmiterlimit{7.0}
\pgfpathmoveto{\pgfqpoint{0.682cm}{1.065cm}}
\pgfpathlineto{\pgfqpoint{1.246cm}{0.315cm}}
\pgfpathlineto{\pgfqpoint{1.811cm}{1.065cm}}
\pgfusepath{stroke}
\end{pgfscope}
\pgfpathmoveto{\pgfqpoint{1.948cm}{1.065cm}}
\pgfpathcurveto{\pgfqpoint{1.948cm}{1.101cm}}{\pgfqpoint{1.933cm}{1.136cm}}{\pgfqpoint{1.907cm}{1.162cm}}
\pgfpathcurveto{\pgfqpoint{1.882cm}{1.187cm}}{\pgfqpoint{1.847cm}{1.202cm}}{\pgfqpoint{1.811cm}{1.202cm}}
\pgfpathcurveto{\pgfqpoint{1.775cm}{1.202cm}}{\pgfqpoint{1.74cm}{1.187cm}}{\pgfqpoint{1.714cm}{1.162cm}}
\pgfpathcurveto{\pgfqpoint{1.689cm}{1.136cm}}{\pgfqpoint{1.674cm}{1.101cm}}{\pgfqpoint{1.674cm}{1.065cm}}
\pgfpathcurveto{\pgfqpoint{1.674cm}{1.029cm}}{\pgfqpoint{1.689cm}{0.994cm}}{\pgfqpoint{1.714cm}{0.968cm}}
\pgfpathcurveto{\pgfqpoint{1.74cm}{0.942cm}}{\pgfqpoint{1.775cm}{0.928cm}}{\pgfqpoint{1.811cm}{0.928cm}}
\pgfpathcurveto{\pgfqpoint{1.847cm}{0.928cm}}{\pgfqpoint{1.882cm}{0.942cm}}{\pgfqpoint{1.907cm}{0.968cm}}
\pgfpathcurveto{\pgfqpoint{1.933cm}{0.994cm}}{\pgfqpoint{1.948cm}{1.029cm}}{\pgfqpoint{1.948cm}{1.065cm}}
\pgfusepath{fill}
\begin{pgfscope}
\pgfsetdash{}{0cm}
\pgfsetlinewidth{0.818mm}
\pgfsetmiterlimit{4.0}
\pgfpathmoveto{\pgfqpoint{1.383cm}{0.178cm}}
\pgfpathcurveto{\pgfqpoint{1.383cm}{0.214cm}}{\pgfqpoint{1.369cm}{0.249cm}}{\pgfqpoint{1.343cm}{0.275cm}}
\pgfpathcurveto{\pgfqpoint{1.317cm}{0.3cm}}{\pgfqpoint{1.283cm}{0.315cm}}{\pgfqpoint{1.246cm}{0.315cm}}
\pgfpathcurveto{\pgfqpoint{1.21cm}{0.315cm}}{\pgfqpoint{1.175cm}{0.3cm}}{\pgfqpoint{1.15cm}{0.275cm}}
\pgfpathcurveto{\pgfqpoint{1.124cm}{0.249cm}}{\pgfqpoint{1.11cm}{0.214cm}}{\pgfqpoint{1.11cm}{0.178cm}}
\pgfpathcurveto{\pgfqpoint{1.11cm}{0.141cm}}{\pgfqpoint{1.124cm}{0.107cm}}{\pgfqpoint{1.15cm}{0.081cm}}
\pgfpathcurveto{\pgfqpoint{1.175cm}{0.055cm}}{\pgfqpoint{1.21cm}{0.041cm}}{\pgfqpoint{1.246cm}{0.041cm}}
\pgfpathcurveto{\pgfqpoint{1.283cm}{0.041cm}}{\pgfqpoint{1.317cm}{0.055cm}}{\pgfqpoint{1.343cm}{0.081cm}}
\pgfpathcurveto{\pgfqpoint{1.369cm}{0.107cm}}{\pgfqpoint{1.383cm}{0.141cm}}{\pgfqpoint{1.383cm}{0.178cm}}
\pgfusepath{stroke}
\end{pgfscope}
\end{pgfscope}
\end{pgfscope}
\end{pgfscope}
\end{tikzpicture}}}\|_{\CC^{-\kappa}(\rho^{\sigma})}
\end{align*}
\begin{align*}
\|3\UU_\leqslant X\|_{\CC^\gamma(\rho^{1+\gamma})}\lesssim 2^{(\frac{1}{2}+\gamma+\kappa)\frac{4}{3}K}\|X\|_{\CC^{-\frac{1}{2}-\kappa}(\rho^{\sigma})}
\end{align*}
This leads to
\begin{align*}
\begin{aligned}
\|3\UU_\leqslant\llbracket X^2 \rrbracket\succ(\phi+\psi)\|_{\CC^\gamma(\rho^{3+\gamma})}&\lesssim \|\phi+\psi\|_{L^\infty(\rho)}\|\UU_\leqslant\llbracket X^2 \rrbracket\|_{\CC^\gamma(\rho^{2+\gamma})}\\
&\lesssim 2^{(1+\gamma+\kappa)K}\|\phi+\psi\|_{L^\infty(\rho)}\lesssim 1+\|\psi\|^{2+\varepsilon}_{L^\infty(\rho)},
\end{aligned}
\end{align*}
\begin{align*}
\begin{aligned}
\|3\UU_\leqslant\llbracket X^2 \rrbracket\prec(\phi+\psi)\|_{\CC^\gamma(\rho^{3+\gamma})}&\lesssim\|\phi+\psi\|_{\CC^{\gamma}(\rho^{1+\gamma})} \|\UU_\leqslant\llbracket X^2 \rrbracket\|_{\CC^\gamma(\rho^{2})}\\
&\lesssim 2^{(1+\gamma+\kappa)K}\|\phi+\psi\|_{\CC^{\gamma}(\rho^{1+\gamma})} \lesssim (1+\|\psi\|^{1+\varepsilon}_{L^\infty(\rho)})(1+\|\psi\|_{\CC^{\gamma}(\rho^{1+\gamma})})\\
&\lesssim 1+ \|\psi\|^{1+\varepsilon}_{L^\infty(\rho)} +\|\psi\|^{1+\varepsilon}_{L^\infty(\rho)}\|\psi\|_{\CC^{\gamma}(\rho^{1+\gamma})},
\end{aligned}
\end{align*}
\begin{align*}
\|3( \phi + \psi)\prec\UU_\leqslantX^{\!\resizebox{!}{.8em}{
\begin{tikzpicture}
\pgfpathmoveto{\pgfqpoint{0cm}{-0.035cm}}
\pgfpathlineto{\pgfqpoint{1.976cm}{-0.035cm}}
\pgfpathlineto{\pgfqpoint{1.976cm}{1.94cm}}
\pgfpathlineto{\pgfqpoint{0cm}{1.94cm}}
\pgfpathclose
\pgfusepath{clip}
\begin{pgfscope}
\begin{pgfscope}
\pgfpathmoveto{\pgfqpoint{0cm}{-0.035cm}}
\pgfpathlineto{\pgfqpoint{1.976cm}{-0.035cm}}
\pgfpathlineto{\pgfqpoint{1.976cm}{1.94cm}}
\pgfpathlineto{\pgfqpoint{0cm}{1.94cm}}
\pgfpathclose
\pgfusepath{clip}
\begin{pgfscope}
\begin{pgfscope}
\pgfsetdash{}{0cm}
\pgfsetlinewidth{0.818mm}
\pgfsetroundcap
\pgfsetroundjoin
\pgfsetmiterlimit{7.0}
\definecolor{eps2pgf_color}{gray}{0}\pgfsetstrokecolor{eps2pgf_color}\pgfsetfillcolor{eps2pgf_color}
\pgfpathmoveto{\pgfqpoint{0.117cm}{1.815cm}}
\pgfpathlineto{\pgfqpoint{0.682cm}{1.065cm}}
\pgfpathlineto{\pgfqpoint{1.246cm}{1.815cm}}
\pgfusepath{stroke}
\end{pgfscope}
\definecolor{eps2pgf_color}{gray}{0}\pgfsetstrokecolor{eps2pgf_color}\pgfsetfillcolor{eps2pgf_color}
\pgfpathmoveto{\pgfqpoint{0.273cm}{1.789cm}}
\pgfpathcurveto{\pgfqpoint{0.273cm}{1.825cm}}{\pgfqpoint{0.259cm}{1.86cm}}{\pgfqpoint{0.233cm}{1.886cm}}
\pgfpathcurveto{\pgfqpoint{0.207cm}{1.912cm}}{\pgfqpoint{0.173cm}{1.926cm}}{\pgfqpoint{0.137cm}{1.926cm}}
\pgfpathcurveto{\pgfqpoint{0.1cm}{1.926cm}}{\pgfqpoint{0.066cm}{1.912cm}}{\pgfqpoint{0.04cm}{1.886cm}}
\pgfpathcurveto{\pgfqpoint{0.014cm}{1.86cm}}{\pgfqpoint{0cm}{1.825cm}}{\pgfqpoint{0cm}{1.789cm}}
\pgfpathcurveto{\pgfqpoint{0cm}{1.753cm}}{\pgfqpoint{0.014cm}{1.718cm}}{\pgfqpoint{0.04cm}{1.692cm}}
\pgfpathcurveto{\pgfqpoint{0.066cm}{1.667cm}}{\pgfqpoint{0.1cm}{1.652cm}}{\pgfqpoint{0.137cm}{1.652cm}}
\pgfpathcurveto{\pgfqpoint{0.173cm}{1.652cm}}{\pgfqpoint{0.207cm}{1.667cm}}{\pgfqpoint{0.233cm}{1.692cm}}
\pgfpathcurveto{\pgfqpoint{0.259cm}{1.718cm}}{\pgfqpoint{0.273cm}{1.753cm}}{\pgfqpoint{0.273cm}{1.789cm}}
\pgfusepath{fill}
\pgfpathmoveto{\pgfqpoint{1.345cm}{1.765cm}}
\pgfpathcurveto{\pgfqpoint{1.345cm}{1.801cm}}{\pgfqpoint{1.331cm}{1.836cm}}{\pgfqpoint{1.305cm}{1.862cm}}
\pgfpathcurveto{\pgfqpoint{1.28cm}{1.887cm}}{\pgfqpoint{1.245cm}{1.902cm}}{\pgfqpoint{1.209cm}{1.902cm}}
\pgfpathcurveto{\pgfqpoint{1.172cm}{1.902cm}}{\pgfqpoint{1.138cm}{1.887cm}}{\pgfqpoint{1.112cm}{1.862cm}}
\pgfpathcurveto{\pgfqpoint{1.087cm}{1.836cm}}{\pgfqpoint{1.072cm}{1.801cm}}{\pgfqpoint{1.072cm}{1.765cm}}
\pgfpathcurveto{\pgfqpoint{1.072cm}{1.728cm}}{\pgfqpoint{1.087cm}{1.694cm}}{\pgfqpoint{1.112cm}{1.668cm}}
\pgfpathcurveto{\pgfqpoint{1.138cm}{1.642cm}}{\pgfqpoint{1.172cm}{1.628cm}}{\pgfqpoint{1.209cm}{1.628cm}}
\pgfpathcurveto{\pgfqpoint{1.245cm}{1.628cm}}{\pgfqpoint{1.28cm}{1.642cm}}{\pgfqpoint{1.305cm}{1.668cm}}
\pgfpathcurveto{\pgfqpoint{1.331cm}{1.694cm}}{\pgfqpoint{1.345cm}{1.728cm}}{\pgfqpoint{1.345cm}{1.765cm}}
\pgfusepath{fill}
\begin{pgfscope}
\pgfsetdash{}{0cm}
\pgfsetlinewidth{0.818mm}
\pgfsetroundcap
\pgfsetroundjoin
\pgfsetmiterlimit{7.0}
\pgfpathmoveto{\pgfqpoint{0.682cm}{1.065cm}}
\pgfpathlineto{\pgfqpoint{1.246cm}{0.315cm}}
\pgfpathlineto{\pgfqpoint{1.811cm}{1.065cm}}
\pgfusepath{stroke}
\end{pgfscope}
\pgfpathmoveto{\pgfqpoint{1.948cm}{1.065cm}}
\pgfpathcurveto{\pgfqpoint{1.948cm}{1.101cm}}{\pgfqpoint{1.933cm}{1.136cm}}{\pgfqpoint{1.907cm}{1.162cm}}
\pgfpathcurveto{\pgfqpoint{1.882cm}{1.187cm}}{\pgfqpoint{1.847cm}{1.202cm}}{\pgfqpoint{1.811cm}{1.202cm}}
\pgfpathcurveto{\pgfqpoint{1.775cm}{1.202cm}}{\pgfqpoint{1.74cm}{1.187cm}}{\pgfqpoint{1.714cm}{1.162cm}}
\pgfpathcurveto{\pgfqpoint{1.689cm}{1.136cm}}{\pgfqpoint{1.674cm}{1.101cm}}{\pgfqpoint{1.674cm}{1.065cm}}
\pgfpathcurveto{\pgfqpoint{1.674cm}{1.029cm}}{\pgfqpoint{1.689cm}{0.994cm}}{\pgfqpoint{1.714cm}{0.968cm}}
\pgfpathcurveto{\pgfqpoint{1.74cm}{0.942cm}}{\pgfqpoint{1.775cm}{0.928cm}}{\pgfqpoint{1.811cm}{0.928cm}}
\pgfpathcurveto{\pgfqpoint{1.847cm}{0.928cm}}{\pgfqpoint{1.882cm}{0.942cm}}{\pgfqpoint{1.907cm}{0.968cm}}
\pgfpathcurveto{\pgfqpoint{1.933cm}{0.994cm}}{\pgfqpoint{1.948cm}{1.029cm}}{\pgfqpoint{1.948cm}{1.065cm}}
\pgfusepath{fill}
\begin{pgfscope}
\pgfsetdash{}{0cm}
\pgfsetlinewidth{0.818mm}
\pgfsetmiterlimit{7.0}
\pgfpathmoveto{\pgfqpoint{1.246cm}{0.315cm}}
\pgfpathlineto{\pgfqpoint{1.244cm}{1.061cm}}
\pgfusepath{stroke}
\end{pgfscope}
\pgfpathmoveto{\pgfqpoint{1.38cm}{1.065cm}}
\pgfpathcurveto{\pgfqpoint{1.38cm}{1.101cm}}{\pgfqpoint{1.366cm}{1.136cm}}{\pgfqpoint{1.34cm}{1.162cm}}
\pgfpathcurveto{\pgfqpoint{1.315cm}{1.187cm}}{\pgfqpoint{1.28cm}{1.202cm}}{\pgfqpoint{1.244cm}{1.202cm}}
\pgfpathcurveto{\pgfqpoint{1.207cm}{1.202cm}}{\pgfqpoint{1.173cm}{1.187cm}}{\pgfqpoint{1.147cm}{1.162cm}}
\pgfpathcurveto{\pgfqpoint{1.121cm}{1.136cm}}{\pgfqpoint{1.107cm}{1.101cm}}{\pgfqpoint{1.107cm}{1.065cm}}
\pgfpathcurveto{\pgfqpoint{1.107cm}{1.029cm}}{\pgfqpoint{1.121cm}{0.994cm}}{\pgfqpoint{1.147cm}{0.968cm}}
\pgfpathcurveto{\pgfqpoint{1.173cm}{0.942cm}}{\pgfqpoint{1.207cm}{0.928cm}}{\pgfqpoint{1.244cm}{0.928cm}}
\pgfpathcurveto{\pgfqpoint{1.28cm}{0.928cm}}{\pgfqpoint{1.315cm}{0.942cm}}{\pgfqpoint{1.34cm}{0.968cm}}
\pgfpathcurveto{\pgfqpoint{1.366cm}{0.994cm}}{\pgfqpoint{1.38cm}{1.029cm}}{\pgfqpoint{1.38cm}{1.065cm}}
\pgfusepath{fill}
\begin{pgfscope}
\pgfsetdash{}{0cm}
\pgfsetlinewidth{0.818mm}
\pgfsetmiterlimit{4.0}
\pgfpathmoveto{\pgfqpoint{1.383cm}{0.178cm}}
\pgfpathcurveto{\pgfqpoint{1.383cm}{0.214cm}}{\pgfqpoint{1.369cm}{0.249cm}}{\pgfqpoint{1.343cm}{0.275cm}}
\pgfpathcurveto{\pgfqpoint{1.317cm}{0.3cm}}{\pgfqpoint{1.283cm}{0.315cm}}{\pgfqpoint{1.246cm}{0.315cm}}
\pgfpathcurveto{\pgfqpoint{1.21cm}{0.315cm}}{\pgfqpoint{1.175cm}{0.3cm}}{\pgfqpoint{1.15cm}{0.275cm}}
\pgfpathcurveto{\pgfqpoint{1.124cm}{0.249cm}}{\pgfqpoint{1.11cm}{0.214cm}}{\pgfqpoint{1.11cm}{0.178cm}}
\pgfpathcurveto{\pgfqpoint{1.11cm}{0.141cm}}{\pgfqpoint{1.124cm}{0.107cm}}{\pgfqpoint{1.15cm}{0.081cm}}
\pgfpathcurveto{\pgfqpoint{1.175cm}{0.055cm}}{\pgfqpoint{1.21cm}{0.041cm}}{\pgfqpoint{1.246cm}{0.041cm}}
\pgfpathcurveto{\pgfqpoint{1.283cm}{0.041cm}}{\pgfqpoint{1.317cm}{0.055cm}}{\pgfqpoint{1.343cm}{0.081cm}}
\pgfpathcurveto{\pgfqpoint{1.369cm}{0.107cm}}{\pgfqpoint{1.383cm}{0.141cm}}{\pgfqpoint{1.383cm}{0.178cm}}
\pgfusepath{stroke}
\end{pgfscope}
\end{pgfscope}
\end{pgfscope}
\end{pgfscope}
\end{tikzpicture}}}\|_{\CC^\gamma(\rho^{3+\gamma})}&\lesssim \|\phi+\psi\|_{L^\infty(\rho)}\|\UU_\leqslant X^{\!\resizebox{!}{.8em}{
\begin{tikzpicture}
\pgfpathmoveto{\pgfqpoint{0cm}{-0.035cm}}
\pgfpathlineto{\pgfqpoint{1.976cm}{-0.035cm}}
\pgfpathlineto{\pgfqpoint{1.976cm}{1.94cm}}
\pgfpathlineto{\pgfqpoint{0cm}{1.94cm}}
\pgfpathclose
\pgfusepath{clip}
\begin{pgfscope}
\begin{pgfscope}
\pgfpathmoveto{\pgfqpoint{0cm}{-0.035cm}}
\pgfpathlineto{\pgfqpoint{1.976cm}{-0.035cm}}
\pgfpathlineto{\pgfqpoint{1.976cm}{1.94cm}}
\pgfpathlineto{\pgfqpoint{0cm}{1.94cm}}
\pgfpathclose
\pgfusepath{clip}
\begin{pgfscope}
\begin{pgfscope}
\pgfsetdash{}{0cm}
\pgfsetlinewidth{0.818mm}
\pgfsetroundcap
\pgfsetroundjoin
\pgfsetmiterlimit{7.0}
\definecolor{eps2pgf_color}{gray}{0}\pgfsetstrokecolor{eps2pgf_color}\pgfsetfillcolor{eps2pgf_color}
\pgfpathmoveto{\pgfqpoint{0.117cm}{1.815cm}}
\pgfpathlineto{\pgfqpoint{0.682cm}{1.065cm}}
\pgfpathlineto{\pgfqpoint{1.246cm}{1.815cm}}
\pgfusepath{stroke}
\end{pgfscope}
\definecolor{eps2pgf_color}{gray}{0}\pgfsetstrokecolor{eps2pgf_color}\pgfsetfillcolor{eps2pgf_color}
\pgfpathmoveto{\pgfqpoint{0.273cm}{1.789cm}}
\pgfpathcurveto{\pgfqpoint{0.273cm}{1.825cm}}{\pgfqpoint{0.259cm}{1.86cm}}{\pgfqpoint{0.233cm}{1.886cm}}
\pgfpathcurveto{\pgfqpoint{0.207cm}{1.912cm}}{\pgfqpoint{0.173cm}{1.926cm}}{\pgfqpoint{0.137cm}{1.926cm}}
\pgfpathcurveto{\pgfqpoint{0.1cm}{1.926cm}}{\pgfqpoint{0.066cm}{1.912cm}}{\pgfqpoint{0.04cm}{1.886cm}}
\pgfpathcurveto{\pgfqpoint{0.014cm}{1.86cm}}{\pgfqpoint{0cm}{1.825cm}}{\pgfqpoint{0cm}{1.789cm}}
\pgfpathcurveto{\pgfqpoint{0cm}{1.753cm}}{\pgfqpoint{0.014cm}{1.718cm}}{\pgfqpoint{0.04cm}{1.692cm}}
\pgfpathcurveto{\pgfqpoint{0.066cm}{1.667cm}}{\pgfqpoint{0.1cm}{1.652cm}}{\pgfqpoint{0.137cm}{1.652cm}}
\pgfpathcurveto{\pgfqpoint{0.173cm}{1.652cm}}{\pgfqpoint{0.207cm}{1.667cm}}{\pgfqpoint{0.233cm}{1.692cm}}
\pgfpathcurveto{\pgfqpoint{0.259cm}{1.718cm}}{\pgfqpoint{0.273cm}{1.753cm}}{\pgfqpoint{0.273cm}{1.789cm}}
\pgfusepath{fill}
\pgfpathmoveto{\pgfqpoint{1.345cm}{1.765cm}}
\pgfpathcurveto{\pgfqpoint{1.345cm}{1.801cm}}{\pgfqpoint{1.331cm}{1.836cm}}{\pgfqpoint{1.305cm}{1.862cm}}
\pgfpathcurveto{\pgfqpoint{1.28cm}{1.887cm}}{\pgfqpoint{1.245cm}{1.902cm}}{\pgfqpoint{1.209cm}{1.902cm}}
\pgfpathcurveto{\pgfqpoint{1.172cm}{1.902cm}}{\pgfqpoint{1.138cm}{1.887cm}}{\pgfqpoint{1.112cm}{1.862cm}}
\pgfpathcurveto{\pgfqpoint{1.087cm}{1.836cm}}{\pgfqpoint{1.072cm}{1.801cm}}{\pgfqpoint{1.072cm}{1.765cm}}
\pgfpathcurveto{\pgfqpoint{1.072cm}{1.728cm}}{\pgfqpoint{1.087cm}{1.694cm}}{\pgfqpoint{1.112cm}{1.668cm}}
\pgfpathcurveto{\pgfqpoint{1.138cm}{1.642cm}}{\pgfqpoint{1.172cm}{1.628cm}}{\pgfqpoint{1.209cm}{1.628cm}}
\pgfpathcurveto{\pgfqpoint{1.245cm}{1.628cm}}{\pgfqpoint{1.28cm}{1.642cm}}{\pgfqpoint{1.305cm}{1.668cm}}
\pgfpathcurveto{\pgfqpoint{1.331cm}{1.694cm}}{\pgfqpoint{1.345cm}{1.728cm}}{\pgfqpoint{1.345cm}{1.765cm}}
\pgfusepath{fill}
\begin{pgfscope}
\pgfsetdash{}{0cm}
\pgfsetlinewidth{0.818mm}
\pgfsetroundcap
\pgfsetroundjoin
\pgfsetmiterlimit{7.0}
\pgfpathmoveto{\pgfqpoint{0.682cm}{1.065cm}}
\pgfpathlineto{\pgfqpoint{1.246cm}{0.315cm}}
\pgfpathlineto{\pgfqpoint{1.811cm}{1.065cm}}
\pgfusepath{stroke}
\end{pgfscope}
\pgfpathmoveto{\pgfqpoint{1.948cm}{1.065cm}}
\pgfpathcurveto{\pgfqpoint{1.948cm}{1.101cm}}{\pgfqpoint{1.933cm}{1.136cm}}{\pgfqpoint{1.907cm}{1.162cm}}
\pgfpathcurveto{\pgfqpoint{1.882cm}{1.187cm}}{\pgfqpoint{1.847cm}{1.202cm}}{\pgfqpoint{1.811cm}{1.202cm}}
\pgfpathcurveto{\pgfqpoint{1.775cm}{1.202cm}}{\pgfqpoint{1.74cm}{1.187cm}}{\pgfqpoint{1.714cm}{1.162cm}}
\pgfpathcurveto{\pgfqpoint{1.689cm}{1.136cm}}{\pgfqpoint{1.674cm}{1.101cm}}{\pgfqpoint{1.674cm}{1.065cm}}
\pgfpathcurveto{\pgfqpoint{1.674cm}{1.029cm}}{\pgfqpoint{1.689cm}{0.994cm}}{\pgfqpoint{1.714cm}{0.968cm}}
\pgfpathcurveto{\pgfqpoint{1.74cm}{0.942cm}}{\pgfqpoint{1.775cm}{0.928cm}}{\pgfqpoint{1.811cm}{0.928cm}}
\pgfpathcurveto{\pgfqpoint{1.847cm}{0.928cm}}{\pgfqpoint{1.882cm}{0.942cm}}{\pgfqpoint{1.907cm}{0.968cm}}
\pgfpathcurveto{\pgfqpoint{1.933cm}{0.994cm}}{\pgfqpoint{1.948cm}{1.029cm}}{\pgfqpoint{1.948cm}{1.065cm}}
\pgfusepath{fill}
\begin{pgfscope}
\pgfsetdash{}{0cm}
\pgfsetlinewidth{0.818mm}
\pgfsetmiterlimit{7.0}
\pgfpathmoveto{\pgfqpoint{1.246cm}{0.315cm}}
\pgfpathlineto{\pgfqpoint{1.244cm}{1.061cm}}
\pgfusepath{stroke}
\end{pgfscope}
\pgfpathmoveto{\pgfqpoint{1.38cm}{1.065cm}}
\pgfpathcurveto{\pgfqpoint{1.38cm}{1.101cm}}{\pgfqpoint{1.366cm}{1.136cm}}{\pgfqpoint{1.34cm}{1.162cm}}
\pgfpathcurveto{\pgfqpoint{1.315cm}{1.187cm}}{\pgfqpoint{1.28cm}{1.202cm}}{\pgfqpoint{1.244cm}{1.202cm}}
\pgfpathcurveto{\pgfqpoint{1.207cm}{1.202cm}}{\pgfqpoint{1.173cm}{1.187cm}}{\pgfqpoint{1.147cm}{1.162cm}}
\pgfpathcurveto{\pgfqpoint{1.121cm}{1.136cm}}{\pgfqpoint{1.107cm}{1.101cm}}{\pgfqpoint{1.107cm}{1.065cm}}
\pgfpathcurveto{\pgfqpoint{1.107cm}{1.029cm}}{\pgfqpoint{1.121cm}{0.994cm}}{\pgfqpoint{1.147cm}{0.968cm}}
\pgfpathcurveto{\pgfqpoint{1.173cm}{0.942cm}}{\pgfqpoint{1.207cm}{0.928cm}}{\pgfqpoint{1.244cm}{0.928cm}}
\pgfpathcurveto{\pgfqpoint{1.28cm}{0.928cm}}{\pgfqpoint{1.315cm}{0.942cm}}{\pgfqpoint{1.34cm}{0.968cm}}
\pgfpathcurveto{\pgfqpoint{1.366cm}{0.994cm}}{\pgfqpoint{1.38cm}{1.029cm}}{\pgfqpoint{1.38cm}{1.065cm}}
\pgfusepath{fill}
\begin{pgfscope}
\pgfsetdash{}{0cm}
\pgfsetlinewidth{0.818mm}
\pgfsetmiterlimit{4.0}
\pgfpathmoveto{\pgfqpoint{1.383cm}{0.178cm}}
\pgfpathcurveto{\pgfqpoint{1.383cm}{0.214cm}}{\pgfqpoint{1.369cm}{0.249cm}}{\pgfqpoint{1.343cm}{0.275cm}}
\pgfpathcurveto{\pgfqpoint{1.317cm}{0.3cm}}{\pgfqpoint{1.283cm}{0.315cm}}{\pgfqpoint{1.246cm}{0.315cm}}
\pgfpathcurveto{\pgfqpoint{1.21cm}{0.315cm}}{\pgfqpoint{1.175cm}{0.3cm}}{\pgfqpoint{1.15cm}{0.275cm}}
\pgfpathcurveto{\pgfqpoint{1.124cm}{0.249cm}}{\pgfqpoint{1.11cm}{0.214cm}}{\pgfqpoint{1.11cm}{0.178cm}}
\pgfpathcurveto{\pgfqpoint{1.11cm}{0.141cm}}{\pgfqpoint{1.124cm}{0.107cm}}{\pgfqpoint{1.15cm}{0.081cm}}
\pgfpathcurveto{\pgfqpoint{1.175cm}{0.055cm}}{\pgfqpoint{1.21cm}{0.041cm}}{\pgfqpoint{1.246cm}{0.041cm}}
\pgfpathcurveto{\pgfqpoint{1.283cm}{0.041cm}}{\pgfqpoint{1.317cm}{0.055cm}}{\pgfqpoint{1.343cm}{0.081cm}}
\pgfpathcurveto{\pgfqpoint{1.369cm}{0.107cm}}{\pgfqpoint{1.383cm}{0.141cm}}{\pgfqpoint{1.383cm}{0.178cm}}
\pgfusepath{stroke}
\end{pgfscope}
\end{pgfscope}
\end{pgfscope}
\end{pgfscope}
\end{tikzpicture}}}\|_{\CC^{\gamma}(\rho^{2+\gamma})}\lesssim 2^{(\gamma+\kappa)K/2}\|\phi+\psi\|_{L^\infty(\rho)}\\
&\lesssim 1+\|\psi\|_{L^\infty(\rho)}^{1+\varepsilon},
\end{align*}
\begin{align*}
\|3( \phi + \psi)\succcurlyeqX^{\!\resizebox{!}{.8em}{
\begin{tikzpicture}
\pgfpathmoveto{\pgfqpoint{0cm}{-0.035cm}}
\pgfpathlineto{\pgfqpoint{1.976cm}{-0.035cm}}
\pgfpathlineto{\pgfqpoint{1.976cm}{1.94cm}}
\pgfpathlineto{\pgfqpoint{0cm}{1.94cm}}
\pgfpathclose
\pgfusepath{clip}
\begin{pgfscope}
\begin{pgfscope}
\pgfpathmoveto{\pgfqpoint{0cm}{-0.035cm}}
\pgfpathlineto{\pgfqpoint{1.976cm}{-0.035cm}}
\pgfpathlineto{\pgfqpoint{1.976cm}{1.94cm}}
\pgfpathlineto{\pgfqpoint{0cm}{1.94cm}}
\pgfpathclose
\pgfusepath{clip}
\begin{pgfscope}
\begin{pgfscope}
\pgfsetdash{}{0cm}
\pgfsetlinewidth{0.818mm}
\pgfsetroundcap
\pgfsetroundjoin
\pgfsetmiterlimit{7.0}
\definecolor{eps2pgf_color}{gray}{0}\pgfsetstrokecolor{eps2pgf_color}\pgfsetfillcolor{eps2pgf_color}
\pgfpathmoveto{\pgfqpoint{0.117cm}{1.815cm}}
\pgfpathlineto{\pgfqpoint{0.682cm}{1.065cm}}
\pgfpathlineto{\pgfqpoint{1.246cm}{1.815cm}}
\pgfusepath{stroke}
\end{pgfscope}
\definecolor{eps2pgf_color}{gray}{0}\pgfsetstrokecolor{eps2pgf_color}\pgfsetfillcolor{eps2pgf_color}
\pgfpathmoveto{\pgfqpoint{0.273cm}{1.789cm}}
\pgfpathcurveto{\pgfqpoint{0.273cm}{1.825cm}}{\pgfqpoint{0.259cm}{1.86cm}}{\pgfqpoint{0.233cm}{1.886cm}}
\pgfpathcurveto{\pgfqpoint{0.207cm}{1.912cm}}{\pgfqpoint{0.173cm}{1.926cm}}{\pgfqpoint{0.137cm}{1.926cm}}
\pgfpathcurveto{\pgfqpoint{0.1cm}{1.926cm}}{\pgfqpoint{0.066cm}{1.912cm}}{\pgfqpoint{0.04cm}{1.886cm}}
\pgfpathcurveto{\pgfqpoint{0.014cm}{1.86cm}}{\pgfqpoint{0cm}{1.825cm}}{\pgfqpoint{0cm}{1.789cm}}
\pgfpathcurveto{\pgfqpoint{0cm}{1.753cm}}{\pgfqpoint{0.014cm}{1.718cm}}{\pgfqpoint{0.04cm}{1.692cm}}
\pgfpathcurveto{\pgfqpoint{0.066cm}{1.667cm}}{\pgfqpoint{0.1cm}{1.652cm}}{\pgfqpoint{0.137cm}{1.652cm}}
\pgfpathcurveto{\pgfqpoint{0.173cm}{1.652cm}}{\pgfqpoint{0.207cm}{1.667cm}}{\pgfqpoint{0.233cm}{1.692cm}}
\pgfpathcurveto{\pgfqpoint{0.259cm}{1.718cm}}{\pgfqpoint{0.273cm}{1.753cm}}{\pgfqpoint{0.273cm}{1.789cm}}
\pgfusepath{fill}
\pgfpathmoveto{\pgfqpoint{1.345cm}{1.765cm}}
\pgfpathcurveto{\pgfqpoint{1.345cm}{1.801cm}}{\pgfqpoint{1.331cm}{1.836cm}}{\pgfqpoint{1.305cm}{1.862cm}}
\pgfpathcurveto{\pgfqpoint{1.28cm}{1.887cm}}{\pgfqpoint{1.245cm}{1.902cm}}{\pgfqpoint{1.209cm}{1.902cm}}
\pgfpathcurveto{\pgfqpoint{1.172cm}{1.902cm}}{\pgfqpoint{1.138cm}{1.887cm}}{\pgfqpoint{1.112cm}{1.862cm}}
\pgfpathcurveto{\pgfqpoint{1.087cm}{1.836cm}}{\pgfqpoint{1.072cm}{1.801cm}}{\pgfqpoint{1.072cm}{1.765cm}}
\pgfpathcurveto{\pgfqpoint{1.072cm}{1.728cm}}{\pgfqpoint{1.087cm}{1.694cm}}{\pgfqpoint{1.112cm}{1.668cm}}
\pgfpathcurveto{\pgfqpoint{1.138cm}{1.642cm}}{\pgfqpoint{1.172cm}{1.628cm}}{\pgfqpoint{1.209cm}{1.628cm}}
\pgfpathcurveto{\pgfqpoint{1.245cm}{1.628cm}}{\pgfqpoint{1.28cm}{1.642cm}}{\pgfqpoint{1.305cm}{1.668cm}}
\pgfpathcurveto{\pgfqpoint{1.331cm}{1.694cm}}{\pgfqpoint{1.345cm}{1.728cm}}{\pgfqpoint{1.345cm}{1.765cm}}
\pgfusepath{fill}
\begin{pgfscope}
\pgfsetdash{}{0cm}
\pgfsetlinewidth{0.818mm}
\pgfsetroundcap
\pgfsetroundjoin
\pgfsetmiterlimit{7.0}
\pgfpathmoveto{\pgfqpoint{0.682cm}{1.065cm}}
\pgfpathlineto{\pgfqpoint{1.246cm}{0.315cm}}
\pgfpathlineto{\pgfqpoint{1.811cm}{1.065cm}}
\pgfusepath{stroke}
\end{pgfscope}
\pgfpathmoveto{\pgfqpoint{1.948cm}{1.065cm}}
\pgfpathcurveto{\pgfqpoint{1.948cm}{1.101cm}}{\pgfqpoint{1.933cm}{1.136cm}}{\pgfqpoint{1.907cm}{1.162cm}}
\pgfpathcurveto{\pgfqpoint{1.882cm}{1.187cm}}{\pgfqpoint{1.847cm}{1.202cm}}{\pgfqpoint{1.811cm}{1.202cm}}
\pgfpathcurveto{\pgfqpoint{1.775cm}{1.202cm}}{\pgfqpoint{1.74cm}{1.187cm}}{\pgfqpoint{1.714cm}{1.162cm}}
\pgfpathcurveto{\pgfqpoint{1.689cm}{1.136cm}}{\pgfqpoint{1.674cm}{1.101cm}}{\pgfqpoint{1.674cm}{1.065cm}}
\pgfpathcurveto{\pgfqpoint{1.674cm}{1.029cm}}{\pgfqpoint{1.689cm}{0.994cm}}{\pgfqpoint{1.714cm}{0.968cm}}
\pgfpathcurveto{\pgfqpoint{1.74cm}{0.942cm}}{\pgfqpoint{1.775cm}{0.928cm}}{\pgfqpoint{1.811cm}{0.928cm}}
\pgfpathcurveto{\pgfqpoint{1.847cm}{0.928cm}}{\pgfqpoint{1.882cm}{0.942cm}}{\pgfqpoint{1.907cm}{0.968cm}}
\pgfpathcurveto{\pgfqpoint{1.933cm}{0.994cm}}{\pgfqpoint{1.948cm}{1.029cm}}{\pgfqpoint{1.948cm}{1.065cm}}
\pgfusepath{fill}
\begin{pgfscope}
\pgfsetdash{}{0cm}
\pgfsetlinewidth{0.818mm}
\pgfsetmiterlimit{7.0}
\pgfpathmoveto{\pgfqpoint{1.246cm}{0.315cm}}
\pgfpathlineto{\pgfqpoint{1.244cm}{1.061cm}}
\pgfusepath{stroke}
\end{pgfscope}
\pgfpathmoveto{\pgfqpoint{1.38cm}{1.065cm}}
\pgfpathcurveto{\pgfqpoint{1.38cm}{1.101cm}}{\pgfqpoint{1.366cm}{1.136cm}}{\pgfqpoint{1.34cm}{1.162cm}}
\pgfpathcurveto{\pgfqpoint{1.315cm}{1.187cm}}{\pgfqpoint{1.28cm}{1.202cm}}{\pgfqpoint{1.244cm}{1.202cm}}
\pgfpathcurveto{\pgfqpoint{1.207cm}{1.202cm}}{\pgfqpoint{1.173cm}{1.187cm}}{\pgfqpoint{1.147cm}{1.162cm}}
\pgfpathcurveto{\pgfqpoint{1.121cm}{1.136cm}}{\pgfqpoint{1.107cm}{1.101cm}}{\pgfqpoint{1.107cm}{1.065cm}}
\pgfpathcurveto{\pgfqpoint{1.107cm}{1.029cm}}{\pgfqpoint{1.121cm}{0.994cm}}{\pgfqpoint{1.147cm}{0.968cm}}
\pgfpathcurveto{\pgfqpoint{1.173cm}{0.942cm}}{\pgfqpoint{1.207cm}{0.928cm}}{\pgfqpoint{1.244cm}{0.928cm}}
\pgfpathcurveto{\pgfqpoint{1.28cm}{0.928cm}}{\pgfqpoint{1.315cm}{0.942cm}}{\pgfqpoint{1.34cm}{0.968cm}}
\pgfpathcurveto{\pgfqpoint{1.366cm}{0.994cm}}{\pgfqpoint{1.38cm}{1.029cm}}{\pgfqpoint{1.38cm}{1.065cm}}
\pgfusepath{fill}
\begin{pgfscope}
\pgfsetdash{}{0cm}
\pgfsetlinewidth{0.818mm}
\pgfsetmiterlimit{4.0}
\pgfpathmoveto{\pgfqpoint{1.383cm}{0.178cm}}
\pgfpathcurveto{\pgfqpoint{1.383cm}{0.214cm}}{\pgfqpoint{1.369cm}{0.249cm}}{\pgfqpoint{1.343cm}{0.275cm}}
\pgfpathcurveto{\pgfqpoint{1.317cm}{0.3cm}}{\pgfqpoint{1.283cm}{0.315cm}}{\pgfqpoint{1.246cm}{0.315cm}}
\pgfpathcurveto{\pgfqpoint{1.21cm}{0.315cm}}{\pgfqpoint{1.175cm}{0.3cm}}{\pgfqpoint{1.15cm}{0.275cm}}
\pgfpathcurveto{\pgfqpoint{1.124cm}{0.249cm}}{\pgfqpoint{1.11cm}{0.214cm}}{\pgfqpoint{1.11cm}{0.178cm}}
\pgfpathcurveto{\pgfqpoint{1.11cm}{0.141cm}}{\pgfqpoint{1.124cm}{0.107cm}}{\pgfqpoint{1.15cm}{0.081cm}}
\pgfpathcurveto{\pgfqpoint{1.175cm}{0.055cm}}{\pgfqpoint{1.21cm}{0.041cm}}{\pgfqpoint{1.246cm}{0.041cm}}
\pgfpathcurveto{\pgfqpoint{1.283cm}{0.041cm}}{\pgfqpoint{1.317cm}{0.055cm}}{\pgfqpoint{1.343cm}{0.081cm}}
\pgfpathcurveto{\pgfqpoint{1.369cm}{0.107cm}}{\pgfqpoint{1.383cm}{0.141cm}}{\pgfqpoint{1.383cm}{0.178cm}}
\pgfusepath{stroke}
\end{pgfscope}
\end{pgfscope}
\end{pgfscope}
\end{pgfscope}
\end{tikzpicture}}}\|_{\CC^\gamma(\rho^{3+\gamma})}\lesssim \|\phi+\psi\|_{\CC^{\gamma+\kappa}(\rho^{2+\gamma+\kappa})}\lesssim 1+\|\psi\|_{\CC^{1+\alpha}(\rho^{2+\alpha})},
\end{align*}
\begin{align*}
\|6\UU_\leqslant X\succ(X^{\!\resizebox{0.6em}{!}{
\begin{tikzpicture}
\pgfpathmoveto{\pgfqpoint{0cm}{-0.035cm}}
\pgfpathlineto{\pgfqpoint{1.376cm}{-0.035cm}}
\pgfpathlineto{\pgfqpoint{1.376cm}{1.552cm}}
\pgfpathlineto{\pgfqpoint{0cm}{1.552cm}}
\pgfpathclose
\pgfusepath{clip}
\begin{pgfscope}
\begin{pgfscope}
\pgfpathmoveto{\pgfqpoint{0cm}{-0.035cm}}
\pgfpathlineto{\pgfqpoint{1.376cm}{-0.035cm}}
\pgfpathlineto{\pgfqpoint{1.376cm}{1.552cm}}
\pgfpathlineto{\pgfqpoint{0cm}{1.552cm}}
\pgfpathclose
\pgfusepath{clip}
\begin{pgfscope}
\begin{pgfscope}
\pgfsetdash{}{0cm}
\pgfsetlinewidth{0.818mm}
\pgfsetroundcap
\pgfsetroundjoin
\pgfsetmiterlimit{7.0}
\definecolor{eps2pgf_color}{gray}{0}\pgfsetstrokecolor{eps2pgf_color}\pgfsetfillcolor{eps2pgf_color}
\pgfpathmoveto{\pgfqpoint{0.117cm}{1.421cm}}
\pgfpathlineto{\pgfqpoint{0.682cm}{0.671cm}}
\pgfpathlineto{\pgfqpoint{1.246cm}{1.421cm}}
\pgfusepath{stroke}
\end{pgfscope}
\definecolor{eps2pgf_color}{gray}{0}\pgfsetstrokecolor{eps2pgf_color}\pgfsetfillcolor{eps2pgf_color}
\pgfpathmoveto{\pgfqpoint{0.273cm}{1.395cm}}
\pgfpathcurveto{\pgfqpoint{0.273cm}{1.432cm}}{\pgfqpoint{0.259cm}{1.467cm}}{\pgfqpoint{0.233cm}{1.492cm}}
\pgfpathcurveto{\pgfqpoint{0.207cm}{1.518cm}}{\pgfqpoint{0.173cm}{1.532cm}}{\pgfqpoint{0.137cm}{1.532cm}}
\pgfpathcurveto{\pgfqpoint{0.1cm}{1.532cm}}{\pgfqpoint{0.066cm}{1.518cm}}{\pgfqpoint{0.04cm}{1.492cm}}
\pgfpathcurveto{\pgfqpoint{0.014cm}{1.467cm}}{\pgfqpoint{0cm}{1.432cm}}{\pgfqpoint{0cm}{1.395cm}}
\pgfpathcurveto{\pgfqpoint{0cm}{1.359cm}}{\pgfqpoint{0.014cm}{1.324cm}}{\pgfqpoint{0.04cm}{1.299cm}}
\pgfpathcurveto{\pgfqpoint{0.066cm}{1.273cm}}{\pgfqpoint{0.1cm}{1.258cm}}{\pgfqpoint{0.137cm}{1.258cm}}
\pgfpathcurveto{\pgfqpoint{0.173cm}{1.258cm}}{\pgfqpoint{0.207cm}{1.273cm}}{\pgfqpoint{0.233cm}{1.299cm}}
\pgfpathcurveto{\pgfqpoint{0.259cm}{1.324cm}}{\pgfqpoint{0.273cm}{1.359cm}}{\pgfqpoint{0.273cm}{1.395cm}}
\pgfusepath{fill}
\begin{pgfscope}
\pgfsetdash{}{0cm}
\pgfsetlinewidth{0.818mm}
\pgfsetmiterlimit{7.0}
\pgfpathmoveto{\pgfqpoint{0.682cm}{0.671cm}}
\pgfpathlineto{\pgfqpoint{0.679cm}{1.418cm}}
\pgfusepath{stroke}
\end{pgfscope}
\pgfpathmoveto{\pgfqpoint{0.815cm}{1.399cm}}
\pgfpathcurveto{\pgfqpoint{0.815cm}{1.435cm}}{\pgfqpoint{0.801cm}{1.47cm}}{\pgfqpoint{0.775cm}{1.496cm}}
\pgfpathcurveto{\pgfqpoint{0.75cm}{1.521cm}}{\pgfqpoint{0.715cm}{1.536cm}}{\pgfqpoint{0.679cm}{1.536cm}}
\pgfpathcurveto{\pgfqpoint{0.643cm}{1.536cm}}{\pgfqpoint{0.608cm}{1.521cm}}{\pgfqpoint{0.582cm}{1.496cm}}
\pgfpathcurveto{\pgfqpoint{0.557cm}{1.47cm}}{\pgfqpoint{0.542cm}{1.435cm}}{\pgfqpoint{0.542cm}{1.399cm}}
\pgfpathcurveto{\pgfqpoint{0.542cm}{1.363cm}}{\pgfqpoint{0.557cm}{1.328cm}}{\pgfqpoint{0.582cm}{1.302cm}}
\pgfpathcurveto{\pgfqpoint{0.608cm}{1.276cm}}{\pgfqpoint{0.643cm}{1.262cm}}{\pgfqpoint{0.679cm}{1.262cm}}
\pgfpathcurveto{\pgfqpoint{0.715cm}{1.262cm}}{\pgfqpoint{0.75cm}{1.276cm}}{\pgfqpoint{0.775cm}{1.302cm}}
\pgfpathcurveto{\pgfqpoint{0.801cm}{1.328cm}}{\pgfqpoint{0.815cm}{1.363cm}}{\pgfqpoint{0.815cm}{1.399cm}}
\pgfusepath{fill}
\pgfpathmoveto{\pgfqpoint{1.345cm}{1.371cm}}
\pgfpathcurveto{\pgfqpoint{1.345cm}{1.408cm}}{\pgfqpoint{1.331cm}{1.442cm}}{\pgfqpoint{1.305cm}{1.468cm}}
\pgfpathcurveto{\pgfqpoint{1.28cm}{1.494cm}}{\pgfqpoint{1.245cm}{1.508cm}}{\pgfqpoint{1.209cm}{1.508cm}}
\pgfpathcurveto{\pgfqpoint{1.172cm}{1.508cm}}{\pgfqpoint{1.138cm}{1.494cm}}{\pgfqpoint{1.112cm}{1.468cm}}
\pgfpathcurveto{\pgfqpoint{1.087cm}{1.442cm}}{\pgfqpoint{1.072cm}{1.408cm}}{\pgfqpoint{1.072cm}{1.371cm}}
\pgfpathcurveto{\pgfqpoint{1.072cm}{1.335cm}}{\pgfqpoint{1.087cm}{1.3cm}}{\pgfqpoint{1.112cm}{1.274cm}}
\pgfpathcurveto{\pgfqpoint{1.138cm}{1.249cm}}{\pgfqpoint{1.172cm}{1.234cm}}{\pgfqpoint{1.209cm}{1.234cm}}
\pgfpathcurveto{\pgfqpoint{1.245cm}{1.234cm}}{\pgfqpoint{1.28cm}{1.249cm}}{\pgfqpoint{1.305cm}{1.274cm}}
\pgfpathcurveto{\pgfqpoint{1.331cm}{1.3cm}}{\pgfqpoint{1.345cm}{1.335cm}}{\pgfqpoint{1.345cm}{1.371cm}}
\pgfusepath{fill}
\begin{pgfscope}
\pgfsetdash{}{0cm}
\pgfsetlinewidth{0.818mm}
\pgfsetroundcap
\pgfsetmiterlimit{4.0}
\pgfpathmoveto{\pgfqpoint{0.682cm}{0.671cm}}
\pgfpathlineto{\pgfqpoint{0.682cm}{0.042cm}}
\pgfusepath{stroke}
\end{pgfscope}
\end{pgfscope}
\end{pgfscope}
\end{pgfscope}
\end{tikzpicture}}}(\phi+\psi))\|_{\CC^\gamma(\rho^{3+\gamma})}&\lesssim \|\phi+\psi\|_{L^\infty(\rho)}\|\UU_\leqslant X\|_{\CC^\gamma(\rho^{2})}\lesssim 2^{(\frac{1}{2}+\gamma+\kappa)\frac{2}{3}K}\|\phi+\psi\|_{L^\infty(\rho)}\\
&\lesssim 1+\|\psi\|_{L^\infty(\rho)}^{1+\varepsilon},
\end{align*}
\begin{align*}
\|6\UU_\leqslant X\prec(X^{\!\resizebox{0.6em}{!}{
\begin{tikzpicture}
\pgfpathmoveto{\pgfqpoint{0cm}{-0.035cm}}
\pgfpathlineto{\pgfqpoint{1.376cm}{-0.035cm}}
\pgfpathlineto{\pgfqpoint{1.376cm}{1.552cm}}
\pgfpathlineto{\pgfqpoint{0cm}{1.552cm}}
\pgfpathclose
\pgfusepath{clip}
\begin{pgfscope}
\begin{pgfscope}
\pgfpathmoveto{\pgfqpoint{0cm}{-0.035cm}}
\pgfpathlineto{\pgfqpoint{1.376cm}{-0.035cm}}
\pgfpathlineto{\pgfqpoint{1.376cm}{1.552cm}}
\pgfpathlineto{\pgfqpoint{0cm}{1.552cm}}
\pgfpathclose
\pgfusepath{clip}
\begin{pgfscope}
\begin{pgfscope}
\pgfsetdash{}{0cm}
\pgfsetlinewidth{0.818mm}
\pgfsetroundcap
\pgfsetroundjoin
\pgfsetmiterlimit{7.0}
\definecolor{eps2pgf_color}{gray}{0}\pgfsetstrokecolor{eps2pgf_color}\pgfsetfillcolor{eps2pgf_color}
\pgfpathmoveto{\pgfqpoint{0.117cm}{1.421cm}}
\pgfpathlineto{\pgfqpoint{0.682cm}{0.671cm}}
\pgfpathlineto{\pgfqpoint{1.246cm}{1.421cm}}
\pgfusepath{stroke}
\end{pgfscope}
\definecolor{eps2pgf_color}{gray}{0}\pgfsetstrokecolor{eps2pgf_color}\pgfsetfillcolor{eps2pgf_color}
\pgfpathmoveto{\pgfqpoint{0.273cm}{1.395cm}}
\pgfpathcurveto{\pgfqpoint{0.273cm}{1.432cm}}{\pgfqpoint{0.259cm}{1.467cm}}{\pgfqpoint{0.233cm}{1.492cm}}
\pgfpathcurveto{\pgfqpoint{0.207cm}{1.518cm}}{\pgfqpoint{0.173cm}{1.532cm}}{\pgfqpoint{0.137cm}{1.532cm}}
\pgfpathcurveto{\pgfqpoint{0.1cm}{1.532cm}}{\pgfqpoint{0.066cm}{1.518cm}}{\pgfqpoint{0.04cm}{1.492cm}}
\pgfpathcurveto{\pgfqpoint{0.014cm}{1.467cm}}{\pgfqpoint{0cm}{1.432cm}}{\pgfqpoint{0cm}{1.395cm}}
\pgfpathcurveto{\pgfqpoint{0cm}{1.359cm}}{\pgfqpoint{0.014cm}{1.324cm}}{\pgfqpoint{0.04cm}{1.299cm}}
\pgfpathcurveto{\pgfqpoint{0.066cm}{1.273cm}}{\pgfqpoint{0.1cm}{1.258cm}}{\pgfqpoint{0.137cm}{1.258cm}}
\pgfpathcurveto{\pgfqpoint{0.173cm}{1.258cm}}{\pgfqpoint{0.207cm}{1.273cm}}{\pgfqpoint{0.233cm}{1.299cm}}
\pgfpathcurveto{\pgfqpoint{0.259cm}{1.324cm}}{\pgfqpoint{0.273cm}{1.359cm}}{\pgfqpoint{0.273cm}{1.395cm}}
\pgfusepath{fill}
\begin{pgfscope}
\pgfsetdash{}{0cm}
\pgfsetlinewidth{0.818mm}
\pgfsetmiterlimit{7.0}
\pgfpathmoveto{\pgfqpoint{0.682cm}{0.671cm}}
\pgfpathlineto{\pgfqpoint{0.679cm}{1.418cm}}
\pgfusepath{stroke}
\end{pgfscope}
\pgfpathmoveto{\pgfqpoint{0.815cm}{1.399cm}}
\pgfpathcurveto{\pgfqpoint{0.815cm}{1.435cm}}{\pgfqpoint{0.801cm}{1.47cm}}{\pgfqpoint{0.775cm}{1.496cm}}
\pgfpathcurveto{\pgfqpoint{0.75cm}{1.521cm}}{\pgfqpoint{0.715cm}{1.536cm}}{\pgfqpoint{0.679cm}{1.536cm}}
\pgfpathcurveto{\pgfqpoint{0.643cm}{1.536cm}}{\pgfqpoint{0.608cm}{1.521cm}}{\pgfqpoint{0.582cm}{1.496cm}}
\pgfpathcurveto{\pgfqpoint{0.557cm}{1.47cm}}{\pgfqpoint{0.542cm}{1.435cm}}{\pgfqpoint{0.542cm}{1.399cm}}
\pgfpathcurveto{\pgfqpoint{0.542cm}{1.363cm}}{\pgfqpoint{0.557cm}{1.328cm}}{\pgfqpoint{0.582cm}{1.302cm}}
\pgfpathcurveto{\pgfqpoint{0.608cm}{1.276cm}}{\pgfqpoint{0.643cm}{1.262cm}}{\pgfqpoint{0.679cm}{1.262cm}}
\pgfpathcurveto{\pgfqpoint{0.715cm}{1.262cm}}{\pgfqpoint{0.75cm}{1.276cm}}{\pgfqpoint{0.775cm}{1.302cm}}
\pgfpathcurveto{\pgfqpoint{0.801cm}{1.328cm}}{\pgfqpoint{0.815cm}{1.363cm}}{\pgfqpoint{0.815cm}{1.399cm}}
\pgfusepath{fill}
\pgfpathmoveto{\pgfqpoint{1.345cm}{1.371cm}}
\pgfpathcurveto{\pgfqpoint{1.345cm}{1.408cm}}{\pgfqpoint{1.331cm}{1.442cm}}{\pgfqpoint{1.305cm}{1.468cm}}
\pgfpathcurveto{\pgfqpoint{1.28cm}{1.494cm}}{\pgfqpoint{1.245cm}{1.508cm}}{\pgfqpoint{1.209cm}{1.508cm}}
\pgfpathcurveto{\pgfqpoint{1.172cm}{1.508cm}}{\pgfqpoint{1.138cm}{1.494cm}}{\pgfqpoint{1.112cm}{1.468cm}}
\pgfpathcurveto{\pgfqpoint{1.087cm}{1.442cm}}{\pgfqpoint{1.072cm}{1.408cm}}{\pgfqpoint{1.072cm}{1.371cm}}
\pgfpathcurveto{\pgfqpoint{1.072cm}{1.335cm}}{\pgfqpoint{1.087cm}{1.3cm}}{\pgfqpoint{1.112cm}{1.274cm}}
\pgfpathcurveto{\pgfqpoint{1.138cm}{1.249cm}}{\pgfqpoint{1.172cm}{1.234cm}}{\pgfqpoint{1.209cm}{1.234cm}}
\pgfpathcurveto{\pgfqpoint{1.245cm}{1.234cm}}{\pgfqpoint{1.28cm}{1.249cm}}{\pgfqpoint{1.305cm}{1.274cm}}
\pgfpathcurveto{\pgfqpoint{1.331cm}{1.3cm}}{\pgfqpoint{1.345cm}{1.335cm}}{\pgfqpoint{1.345cm}{1.371cm}}
\pgfusepath{fill}
\begin{pgfscope}
\pgfsetdash{}{0cm}
\pgfsetlinewidth{0.818mm}
\pgfsetroundcap
\pgfsetmiterlimit{4.0}
\pgfpathmoveto{\pgfqpoint{0.682cm}{0.671cm}}
\pgfpathlineto{\pgfqpoint{0.682cm}{0.042cm}}
\pgfusepath{stroke}
\end{pgfscope}
\end{pgfscope}
\end{pgfscope}
\end{pgfscope}
\end{tikzpicture}}}(\phi+\psi))\|_{\CC^\gamma(\rho^{3+\gamma})}&\lesssim \|\UU_{\leq}X\|_{\CC^{-\frac{1}{2}-\kappa}(\rho^{\sigma})}\|\phi+\psi\|_{\CC^{\frac{1}{2}+\gamma+\kappa}(\rho^{2+\gamma+\kappa})}\lesssim \|\phi+\psi\|_{\CC^{\frac{1}{2}+\alpha}(\rho^{2+\alpha})}\\
&\lesssim 1+\|\psi\|_{L^\infty(\rho)}^\varepsilon+ \|\psi\|_{\CC^{1+\alpha}(\rho^{2+\alpha})},
\end{align*}
\begin{align*}
\|6(\phi+\psi)\prec\UU_{\leqslant}X^{\!\resizebox{!}{.8em}{
\begin{tikzpicture}
\pgfpathmoveto{\pgfqpoint{0cm}{-0.035cm}}
\pgfpathlineto{\pgfqpoint{1.976cm}{-0.035cm}}
\pgfpathlineto{\pgfqpoint{1.976cm}{1.94cm}}
\pgfpathlineto{\pgfqpoint{0cm}{1.94cm}}
\pgfpathclose
\pgfusepath{clip}
\begin{pgfscope}
\begin{pgfscope}
\pgfpathmoveto{\pgfqpoint{0cm}{-0.035cm}}
\pgfpathlineto{\pgfqpoint{1.976cm}{-0.035cm}}
\pgfpathlineto{\pgfqpoint{1.976cm}{1.94cm}}
\pgfpathlineto{\pgfqpoint{0cm}{1.94cm}}
\pgfpathclose
\pgfusepath{clip}
\begin{pgfscope}
\begin{pgfscope}
\pgfsetdash{}{0cm}
\pgfsetlinewidth{0.818mm}
\pgfsetroundcap
\pgfsetroundjoin
\pgfsetmiterlimit{7.0}
\definecolor{eps2pgf_color}{gray}{0}\pgfsetstrokecolor{eps2pgf_color}\pgfsetfillcolor{eps2pgf_color}
\pgfpathmoveto{\pgfqpoint{0.117cm}{1.815cm}}
\pgfpathlineto{\pgfqpoint{0.682cm}{1.065cm}}
\pgfpathlineto{\pgfqpoint{1.246cm}{1.815cm}}
\pgfusepath{stroke}
\end{pgfscope}
\definecolor{eps2pgf_color}{gray}{0}\pgfsetstrokecolor{eps2pgf_color}\pgfsetfillcolor{eps2pgf_color}
\pgfpathmoveto{\pgfqpoint{0.273cm}{1.789cm}}
\pgfpathcurveto{\pgfqpoint{0.273cm}{1.825cm}}{\pgfqpoint{0.259cm}{1.86cm}}{\pgfqpoint{0.233cm}{1.886cm}}
\pgfpathcurveto{\pgfqpoint{0.207cm}{1.912cm}}{\pgfqpoint{0.173cm}{1.926cm}}{\pgfqpoint{0.137cm}{1.926cm}}
\pgfpathcurveto{\pgfqpoint{0.1cm}{1.926cm}}{\pgfqpoint{0.066cm}{1.912cm}}{\pgfqpoint{0.04cm}{1.886cm}}
\pgfpathcurveto{\pgfqpoint{0.014cm}{1.86cm}}{\pgfqpoint{0cm}{1.825cm}}{\pgfqpoint{0cm}{1.789cm}}
\pgfpathcurveto{\pgfqpoint{0cm}{1.753cm}}{\pgfqpoint{0.014cm}{1.718cm}}{\pgfqpoint{0.04cm}{1.692cm}}
\pgfpathcurveto{\pgfqpoint{0.066cm}{1.667cm}}{\pgfqpoint{0.1cm}{1.652cm}}{\pgfqpoint{0.137cm}{1.652cm}}
\pgfpathcurveto{\pgfqpoint{0.173cm}{1.652cm}}{\pgfqpoint{0.207cm}{1.667cm}}{\pgfqpoint{0.233cm}{1.692cm}}
\pgfpathcurveto{\pgfqpoint{0.259cm}{1.718cm}}{\pgfqpoint{0.273cm}{1.753cm}}{\pgfqpoint{0.273cm}{1.789cm}}
\pgfusepath{fill}
\begin{pgfscope}
\pgfsetdash{}{0cm}
\pgfsetlinewidth{0.818mm}
\pgfsetmiterlimit{7.0}
\pgfpathmoveto{\pgfqpoint{0.682cm}{1.065cm}}
\pgfpathlineto{\pgfqpoint{0.679cm}{1.812cm}}
\pgfusepath{stroke}
\end{pgfscope}
\pgfpathmoveto{\pgfqpoint{0.815cm}{1.793cm}}
\pgfpathcurveto{\pgfqpoint{0.815cm}{1.829cm}}{\pgfqpoint{0.801cm}{1.864cm}}{\pgfqpoint{0.775cm}{1.89cm}}
\pgfpathcurveto{\pgfqpoint{0.75cm}{1.915cm}}{\pgfqpoint{0.715cm}{1.93cm}}{\pgfqpoint{0.679cm}{1.93cm}}
\pgfpathcurveto{\pgfqpoint{0.643cm}{1.93cm}}{\pgfqpoint{0.608cm}{1.915cm}}{\pgfqpoint{0.582cm}{1.89cm}}
\pgfpathcurveto{\pgfqpoint{0.557cm}{1.864cm}}{\pgfqpoint{0.542cm}{1.829cm}}{\pgfqpoint{0.542cm}{1.793cm}}
\pgfpathcurveto{\pgfqpoint{0.542cm}{1.756cm}}{\pgfqpoint{0.557cm}{1.722cm}}{\pgfqpoint{0.582cm}{1.696cm}}
\pgfpathcurveto{\pgfqpoint{0.608cm}{1.67cm}}{\pgfqpoint{0.643cm}{1.656cm}}{\pgfqpoint{0.679cm}{1.656cm}}
\pgfpathcurveto{\pgfqpoint{0.715cm}{1.656cm}}{\pgfqpoint{0.75cm}{1.67cm}}{\pgfqpoint{0.775cm}{1.696cm}}
\pgfpathcurveto{\pgfqpoint{0.801cm}{1.722cm}}{\pgfqpoint{0.815cm}{1.756cm}}{\pgfqpoint{0.815cm}{1.793cm}}
\pgfusepath{fill}
\pgfpathmoveto{\pgfqpoint{1.345cm}{1.765cm}}
\pgfpathcurveto{\pgfqpoint{1.345cm}{1.801cm}}{\pgfqpoint{1.331cm}{1.836cm}}{\pgfqpoint{1.305cm}{1.862cm}}
\pgfpathcurveto{\pgfqpoint{1.28cm}{1.887cm}}{\pgfqpoint{1.245cm}{1.902cm}}{\pgfqpoint{1.209cm}{1.902cm}}
\pgfpathcurveto{\pgfqpoint{1.172cm}{1.902cm}}{\pgfqpoint{1.138cm}{1.887cm}}{\pgfqpoint{1.112cm}{1.862cm}}
\pgfpathcurveto{\pgfqpoint{1.087cm}{1.836cm}}{\pgfqpoint{1.072cm}{1.801cm}}{\pgfqpoint{1.072cm}{1.765cm}}
\pgfpathcurveto{\pgfqpoint{1.072cm}{1.728cm}}{\pgfqpoint{1.087cm}{1.694cm}}{\pgfqpoint{1.112cm}{1.668cm}}
\pgfpathcurveto{\pgfqpoint{1.138cm}{1.642cm}}{\pgfqpoint{1.172cm}{1.628cm}}{\pgfqpoint{1.209cm}{1.628cm}}
\pgfpathcurveto{\pgfqpoint{1.245cm}{1.628cm}}{\pgfqpoint{1.28cm}{1.642cm}}{\pgfqpoint{1.305cm}{1.668cm}}
\pgfpathcurveto{\pgfqpoint{1.331cm}{1.694cm}}{\pgfqpoint{1.345cm}{1.728cm}}{\pgfqpoint{1.345cm}{1.765cm}}
\pgfusepath{fill}
\begin{pgfscope}
\pgfsetdash{}{0cm}
\pgfsetlinewidth{0.818mm}
\pgfsetroundcap
\pgfsetroundjoin
\pgfsetmiterlimit{7.0}
\pgfpathmoveto{\pgfqpoint{0.682cm}{1.065cm}}
\pgfpathlineto{\pgfqpoint{1.246cm}{0.315cm}}
\pgfpathlineto{\pgfqpoint{1.811cm}{1.065cm}}
\pgfusepath{stroke}
\end{pgfscope}
\pgfpathmoveto{\pgfqpoint{1.948cm}{1.065cm}}
\pgfpathcurveto{\pgfqpoint{1.948cm}{1.101cm}}{\pgfqpoint{1.933cm}{1.136cm}}{\pgfqpoint{1.907cm}{1.162cm}}
\pgfpathcurveto{\pgfqpoint{1.882cm}{1.187cm}}{\pgfqpoint{1.847cm}{1.202cm}}{\pgfqpoint{1.811cm}{1.202cm}}
\pgfpathcurveto{\pgfqpoint{1.775cm}{1.202cm}}{\pgfqpoint{1.74cm}{1.187cm}}{\pgfqpoint{1.714cm}{1.162cm}}
\pgfpathcurveto{\pgfqpoint{1.689cm}{1.136cm}}{\pgfqpoint{1.674cm}{1.101cm}}{\pgfqpoint{1.674cm}{1.065cm}}
\pgfpathcurveto{\pgfqpoint{1.674cm}{1.029cm}}{\pgfqpoint{1.689cm}{0.994cm}}{\pgfqpoint{1.714cm}{0.968cm}}
\pgfpathcurveto{\pgfqpoint{1.74cm}{0.942cm}}{\pgfqpoint{1.775cm}{0.928cm}}{\pgfqpoint{1.811cm}{0.928cm}}
\pgfpathcurveto{\pgfqpoint{1.847cm}{0.928cm}}{\pgfqpoint{1.882cm}{0.942cm}}{\pgfqpoint{1.907cm}{0.968cm}}
\pgfpathcurveto{\pgfqpoint{1.933cm}{0.994cm}}{\pgfqpoint{1.948cm}{1.029cm}}{\pgfqpoint{1.948cm}{1.065cm}}
\pgfusepath{fill}
\begin{pgfscope}
\pgfsetdash{}{0cm}
\pgfsetlinewidth{0.818mm}
\pgfsetmiterlimit{4.0}
\pgfpathmoveto{\pgfqpoint{1.383cm}{0.178cm}}
\pgfpathcurveto{\pgfqpoint{1.383cm}{0.214cm}}{\pgfqpoint{1.369cm}{0.249cm}}{\pgfqpoint{1.343cm}{0.275cm}}
\pgfpathcurveto{\pgfqpoint{1.317cm}{0.3cm}}{\pgfqpoint{1.283cm}{0.315cm}}{\pgfqpoint{1.246cm}{0.315cm}}
\pgfpathcurveto{\pgfqpoint{1.21cm}{0.315cm}}{\pgfqpoint{1.175cm}{0.3cm}}{\pgfqpoint{1.15cm}{0.275cm}}
\pgfpathcurveto{\pgfqpoint{1.124cm}{0.249cm}}{\pgfqpoint{1.11cm}{0.214cm}}{\pgfqpoint{1.11cm}{0.178cm}}
\pgfpathcurveto{\pgfqpoint{1.11cm}{0.141cm}}{\pgfqpoint{1.124cm}{0.107cm}}{\pgfqpoint{1.15cm}{0.081cm}}
\pgfpathcurveto{\pgfqpoint{1.175cm}{0.055cm}}{\pgfqpoint{1.21cm}{0.041cm}}{\pgfqpoint{1.246cm}{0.041cm}}
\pgfpathcurveto{\pgfqpoint{1.283cm}{0.041cm}}{\pgfqpoint{1.317cm}{0.055cm}}{\pgfqpoint{1.343cm}{0.081cm}}
\pgfpathcurveto{\pgfqpoint{1.369cm}{0.107cm}}{\pgfqpoint{1.383cm}{0.141cm}}{\pgfqpoint{1.383cm}{0.178cm}}
\pgfusepath{stroke}
\end{pgfscope}
\end{pgfscope}
\end{pgfscope}
\end{pgfscope}
\end{tikzpicture}}}\|_{\CC^\gamma(\rho^{3+\gamma})}&\lesssim \|\phi+\psi\|_{L^\infty(\rho)}\|\UU_\leqslant X^{\!\resizebox{!}{.8em}{
\begin{tikzpicture}
\pgfpathmoveto{\pgfqpoint{0cm}{-0.035cm}}
\pgfpathlineto{\pgfqpoint{1.976cm}{-0.035cm}}
\pgfpathlineto{\pgfqpoint{1.976cm}{1.94cm}}
\pgfpathlineto{\pgfqpoint{0cm}{1.94cm}}
\pgfpathclose
\pgfusepath{clip}
\begin{pgfscope}
\begin{pgfscope}
\pgfpathmoveto{\pgfqpoint{0cm}{-0.035cm}}
\pgfpathlineto{\pgfqpoint{1.976cm}{-0.035cm}}
\pgfpathlineto{\pgfqpoint{1.976cm}{1.94cm}}
\pgfpathlineto{\pgfqpoint{0cm}{1.94cm}}
\pgfpathclose
\pgfusepath{clip}
\begin{pgfscope}
\begin{pgfscope}
\pgfsetdash{}{0cm}
\pgfsetlinewidth{0.818mm}
\pgfsetroundcap
\pgfsetroundjoin
\pgfsetmiterlimit{7.0}
\definecolor{eps2pgf_color}{gray}{0}\pgfsetstrokecolor{eps2pgf_color}\pgfsetfillcolor{eps2pgf_color}
\pgfpathmoveto{\pgfqpoint{0.117cm}{1.815cm}}
\pgfpathlineto{\pgfqpoint{0.682cm}{1.065cm}}
\pgfpathlineto{\pgfqpoint{1.246cm}{1.815cm}}
\pgfusepath{stroke}
\end{pgfscope}
\definecolor{eps2pgf_color}{gray}{0}\pgfsetstrokecolor{eps2pgf_color}\pgfsetfillcolor{eps2pgf_color}
\pgfpathmoveto{\pgfqpoint{0.273cm}{1.789cm}}
\pgfpathcurveto{\pgfqpoint{0.273cm}{1.825cm}}{\pgfqpoint{0.259cm}{1.86cm}}{\pgfqpoint{0.233cm}{1.886cm}}
\pgfpathcurveto{\pgfqpoint{0.207cm}{1.912cm}}{\pgfqpoint{0.173cm}{1.926cm}}{\pgfqpoint{0.137cm}{1.926cm}}
\pgfpathcurveto{\pgfqpoint{0.1cm}{1.926cm}}{\pgfqpoint{0.066cm}{1.912cm}}{\pgfqpoint{0.04cm}{1.886cm}}
\pgfpathcurveto{\pgfqpoint{0.014cm}{1.86cm}}{\pgfqpoint{0cm}{1.825cm}}{\pgfqpoint{0cm}{1.789cm}}
\pgfpathcurveto{\pgfqpoint{0cm}{1.753cm}}{\pgfqpoint{0.014cm}{1.718cm}}{\pgfqpoint{0.04cm}{1.692cm}}
\pgfpathcurveto{\pgfqpoint{0.066cm}{1.667cm}}{\pgfqpoint{0.1cm}{1.652cm}}{\pgfqpoint{0.137cm}{1.652cm}}
\pgfpathcurveto{\pgfqpoint{0.173cm}{1.652cm}}{\pgfqpoint{0.207cm}{1.667cm}}{\pgfqpoint{0.233cm}{1.692cm}}
\pgfpathcurveto{\pgfqpoint{0.259cm}{1.718cm}}{\pgfqpoint{0.273cm}{1.753cm}}{\pgfqpoint{0.273cm}{1.789cm}}
\pgfusepath{fill}
\begin{pgfscope}
\pgfsetdash{}{0cm}
\pgfsetlinewidth{0.818mm}
\pgfsetmiterlimit{7.0}
\pgfpathmoveto{\pgfqpoint{0.682cm}{1.065cm}}
\pgfpathlineto{\pgfqpoint{0.679cm}{1.812cm}}
\pgfusepath{stroke}
\end{pgfscope}
\pgfpathmoveto{\pgfqpoint{0.815cm}{1.793cm}}
\pgfpathcurveto{\pgfqpoint{0.815cm}{1.829cm}}{\pgfqpoint{0.801cm}{1.864cm}}{\pgfqpoint{0.775cm}{1.89cm}}
\pgfpathcurveto{\pgfqpoint{0.75cm}{1.915cm}}{\pgfqpoint{0.715cm}{1.93cm}}{\pgfqpoint{0.679cm}{1.93cm}}
\pgfpathcurveto{\pgfqpoint{0.643cm}{1.93cm}}{\pgfqpoint{0.608cm}{1.915cm}}{\pgfqpoint{0.582cm}{1.89cm}}
\pgfpathcurveto{\pgfqpoint{0.557cm}{1.864cm}}{\pgfqpoint{0.542cm}{1.829cm}}{\pgfqpoint{0.542cm}{1.793cm}}
\pgfpathcurveto{\pgfqpoint{0.542cm}{1.756cm}}{\pgfqpoint{0.557cm}{1.722cm}}{\pgfqpoint{0.582cm}{1.696cm}}
\pgfpathcurveto{\pgfqpoint{0.608cm}{1.67cm}}{\pgfqpoint{0.643cm}{1.656cm}}{\pgfqpoint{0.679cm}{1.656cm}}
\pgfpathcurveto{\pgfqpoint{0.715cm}{1.656cm}}{\pgfqpoint{0.75cm}{1.67cm}}{\pgfqpoint{0.775cm}{1.696cm}}
\pgfpathcurveto{\pgfqpoint{0.801cm}{1.722cm}}{\pgfqpoint{0.815cm}{1.756cm}}{\pgfqpoint{0.815cm}{1.793cm}}
\pgfusepath{fill}
\pgfpathmoveto{\pgfqpoint{1.345cm}{1.765cm}}
\pgfpathcurveto{\pgfqpoint{1.345cm}{1.801cm}}{\pgfqpoint{1.331cm}{1.836cm}}{\pgfqpoint{1.305cm}{1.862cm}}
\pgfpathcurveto{\pgfqpoint{1.28cm}{1.887cm}}{\pgfqpoint{1.245cm}{1.902cm}}{\pgfqpoint{1.209cm}{1.902cm}}
\pgfpathcurveto{\pgfqpoint{1.172cm}{1.902cm}}{\pgfqpoint{1.138cm}{1.887cm}}{\pgfqpoint{1.112cm}{1.862cm}}
\pgfpathcurveto{\pgfqpoint{1.087cm}{1.836cm}}{\pgfqpoint{1.072cm}{1.801cm}}{\pgfqpoint{1.072cm}{1.765cm}}
\pgfpathcurveto{\pgfqpoint{1.072cm}{1.728cm}}{\pgfqpoint{1.087cm}{1.694cm}}{\pgfqpoint{1.112cm}{1.668cm}}
\pgfpathcurveto{\pgfqpoint{1.138cm}{1.642cm}}{\pgfqpoint{1.172cm}{1.628cm}}{\pgfqpoint{1.209cm}{1.628cm}}
\pgfpathcurveto{\pgfqpoint{1.245cm}{1.628cm}}{\pgfqpoint{1.28cm}{1.642cm}}{\pgfqpoint{1.305cm}{1.668cm}}
\pgfpathcurveto{\pgfqpoint{1.331cm}{1.694cm}}{\pgfqpoint{1.345cm}{1.728cm}}{\pgfqpoint{1.345cm}{1.765cm}}
\pgfusepath{fill}
\begin{pgfscope}
\pgfsetdash{}{0cm}
\pgfsetlinewidth{0.818mm}
\pgfsetroundcap
\pgfsetroundjoin
\pgfsetmiterlimit{7.0}
\pgfpathmoveto{\pgfqpoint{0.682cm}{1.065cm}}
\pgfpathlineto{\pgfqpoint{1.246cm}{0.315cm}}
\pgfpathlineto{\pgfqpoint{1.811cm}{1.065cm}}
\pgfusepath{stroke}
\end{pgfscope}
\pgfpathmoveto{\pgfqpoint{1.948cm}{1.065cm}}
\pgfpathcurveto{\pgfqpoint{1.948cm}{1.101cm}}{\pgfqpoint{1.933cm}{1.136cm}}{\pgfqpoint{1.907cm}{1.162cm}}
\pgfpathcurveto{\pgfqpoint{1.882cm}{1.187cm}}{\pgfqpoint{1.847cm}{1.202cm}}{\pgfqpoint{1.811cm}{1.202cm}}
\pgfpathcurveto{\pgfqpoint{1.775cm}{1.202cm}}{\pgfqpoint{1.74cm}{1.187cm}}{\pgfqpoint{1.714cm}{1.162cm}}
\pgfpathcurveto{\pgfqpoint{1.689cm}{1.136cm}}{\pgfqpoint{1.674cm}{1.101cm}}{\pgfqpoint{1.674cm}{1.065cm}}
\pgfpathcurveto{\pgfqpoint{1.674cm}{1.029cm}}{\pgfqpoint{1.689cm}{0.994cm}}{\pgfqpoint{1.714cm}{0.968cm}}
\pgfpathcurveto{\pgfqpoint{1.74cm}{0.942cm}}{\pgfqpoint{1.775cm}{0.928cm}}{\pgfqpoint{1.811cm}{0.928cm}}
\pgfpathcurveto{\pgfqpoint{1.847cm}{0.928cm}}{\pgfqpoint{1.882cm}{0.942cm}}{\pgfqpoint{1.907cm}{0.968cm}}
\pgfpathcurveto{\pgfqpoint{1.933cm}{0.994cm}}{\pgfqpoint{1.948cm}{1.029cm}}{\pgfqpoint{1.948cm}{1.065cm}}
\pgfusepath{fill}
\begin{pgfscope}
\pgfsetdash{}{0cm}
\pgfsetlinewidth{0.818mm}
\pgfsetmiterlimit{4.0}
\pgfpathmoveto{\pgfqpoint{1.383cm}{0.178cm}}
\pgfpathcurveto{\pgfqpoint{1.383cm}{0.214cm}}{\pgfqpoint{1.369cm}{0.249cm}}{\pgfqpoint{1.343cm}{0.275cm}}
\pgfpathcurveto{\pgfqpoint{1.317cm}{0.3cm}}{\pgfqpoint{1.283cm}{0.315cm}}{\pgfqpoint{1.246cm}{0.315cm}}
\pgfpathcurveto{\pgfqpoint{1.21cm}{0.315cm}}{\pgfqpoint{1.175cm}{0.3cm}}{\pgfqpoint{1.15cm}{0.275cm}}
\pgfpathcurveto{\pgfqpoint{1.124cm}{0.249cm}}{\pgfqpoint{1.11cm}{0.214cm}}{\pgfqpoint{1.11cm}{0.178cm}}
\pgfpathcurveto{\pgfqpoint{1.11cm}{0.141cm}}{\pgfqpoint{1.124cm}{0.107cm}}{\pgfqpoint{1.15cm}{0.081cm}}
\pgfpathcurveto{\pgfqpoint{1.175cm}{0.055cm}}{\pgfqpoint{1.21cm}{0.041cm}}{\pgfqpoint{1.246cm}{0.041cm}}
\pgfpathcurveto{\pgfqpoint{1.283cm}{0.041cm}}{\pgfqpoint{1.317cm}{0.055cm}}{\pgfqpoint{1.343cm}{0.081cm}}
\pgfpathcurveto{\pgfqpoint{1.369cm}{0.107cm}}{\pgfqpoint{1.383cm}{0.141cm}}{\pgfqpoint{1.383cm}{0.178cm}}
\pgfusepath{stroke}
\end{pgfscope}
\end{pgfscope}
\end{pgfscope}
\end{pgfscope}
\end{tikzpicture}}}\|_{\CC^\gamma(\rho^{2+\gamma})}\lesssim 2^{(\gamma+\kappa)K/2} \|\phi+\psi\|_{L^\infty(\rho)}\\
&\lesssim  1+\|\psi\|_{L^\infty(\rho)}^{1+\varepsilon},
\end{align*}
\begin{align*}
\|6(\phi+\psi)\succcurlyeqX^{\!\resizebox{!}{.8em}{
\begin{tikzpicture}
\pgfpathmoveto{\pgfqpoint{0cm}{-0.035cm}}
\pgfpathlineto{\pgfqpoint{1.976cm}{-0.035cm}}
\pgfpathlineto{\pgfqpoint{1.976cm}{1.94cm}}
\pgfpathlineto{\pgfqpoint{0cm}{1.94cm}}
\pgfpathclose
\pgfusepath{clip}
\begin{pgfscope}
\begin{pgfscope}
\pgfpathmoveto{\pgfqpoint{0cm}{-0.035cm}}
\pgfpathlineto{\pgfqpoint{1.976cm}{-0.035cm}}
\pgfpathlineto{\pgfqpoint{1.976cm}{1.94cm}}
\pgfpathlineto{\pgfqpoint{0cm}{1.94cm}}
\pgfpathclose
\pgfusepath{clip}
\begin{pgfscope}
\begin{pgfscope}
\pgfsetdash{}{0cm}
\pgfsetlinewidth{0.818mm}
\pgfsetroundcap
\pgfsetroundjoin
\pgfsetmiterlimit{7.0}
\definecolor{eps2pgf_color}{gray}{0}\pgfsetstrokecolor{eps2pgf_color}\pgfsetfillcolor{eps2pgf_color}
\pgfpathmoveto{\pgfqpoint{0.117cm}{1.815cm}}
\pgfpathlineto{\pgfqpoint{0.682cm}{1.065cm}}
\pgfpathlineto{\pgfqpoint{1.246cm}{1.815cm}}
\pgfusepath{stroke}
\end{pgfscope}
\definecolor{eps2pgf_color}{gray}{0}\pgfsetstrokecolor{eps2pgf_color}\pgfsetfillcolor{eps2pgf_color}
\pgfpathmoveto{\pgfqpoint{0.273cm}{1.789cm}}
\pgfpathcurveto{\pgfqpoint{0.273cm}{1.825cm}}{\pgfqpoint{0.259cm}{1.86cm}}{\pgfqpoint{0.233cm}{1.886cm}}
\pgfpathcurveto{\pgfqpoint{0.207cm}{1.912cm}}{\pgfqpoint{0.173cm}{1.926cm}}{\pgfqpoint{0.137cm}{1.926cm}}
\pgfpathcurveto{\pgfqpoint{0.1cm}{1.926cm}}{\pgfqpoint{0.066cm}{1.912cm}}{\pgfqpoint{0.04cm}{1.886cm}}
\pgfpathcurveto{\pgfqpoint{0.014cm}{1.86cm}}{\pgfqpoint{0cm}{1.825cm}}{\pgfqpoint{0cm}{1.789cm}}
\pgfpathcurveto{\pgfqpoint{0cm}{1.753cm}}{\pgfqpoint{0.014cm}{1.718cm}}{\pgfqpoint{0.04cm}{1.692cm}}
\pgfpathcurveto{\pgfqpoint{0.066cm}{1.667cm}}{\pgfqpoint{0.1cm}{1.652cm}}{\pgfqpoint{0.137cm}{1.652cm}}
\pgfpathcurveto{\pgfqpoint{0.173cm}{1.652cm}}{\pgfqpoint{0.207cm}{1.667cm}}{\pgfqpoint{0.233cm}{1.692cm}}
\pgfpathcurveto{\pgfqpoint{0.259cm}{1.718cm}}{\pgfqpoint{0.273cm}{1.753cm}}{\pgfqpoint{0.273cm}{1.789cm}}
\pgfusepath{fill}
\begin{pgfscope}
\pgfsetdash{}{0cm}
\pgfsetlinewidth{0.818mm}
\pgfsetmiterlimit{7.0}
\pgfpathmoveto{\pgfqpoint{0.682cm}{1.065cm}}
\pgfpathlineto{\pgfqpoint{0.679cm}{1.812cm}}
\pgfusepath{stroke}
\end{pgfscope}
\pgfpathmoveto{\pgfqpoint{0.815cm}{1.793cm}}
\pgfpathcurveto{\pgfqpoint{0.815cm}{1.829cm}}{\pgfqpoint{0.801cm}{1.864cm}}{\pgfqpoint{0.775cm}{1.89cm}}
\pgfpathcurveto{\pgfqpoint{0.75cm}{1.915cm}}{\pgfqpoint{0.715cm}{1.93cm}}{\pgfqpoint{0.679cm}{1.93cm}}
\pgfpathcurveto{\pgfqpoint{0.643cm}{1.93cm}}{\pgfqpoint{0.608cm}{1.915cm}}{\pgfqpoint{0.582cm}{1.89cm}}
\pgfpathcurveto{\pgfqpoint{0.557cm}{1.864cm}}{\pgfqpoint{0.542cm}{1.829cm}}{\pgfqpoint{0.542cm}{1.793cm}}
\pgfpathcurveto{\pgfqpoint{0.542cm}{1.756cm}}{\pgfqpoint{0.557cm}{1.722cm}}{\pgfqpoint{0.582cm}{1.696cm}}
\pgfpathcurveto{\pgfqpoint{0.608cm}{1.67cm}}{\pgfqpoint{0.643cm}{1.656cm}}{\pgfqpoint{0.679cm}{1.656cm}}
\pgfpathcurveto{\pgfqpoint{0.715cm}{1.656cm}}{\pgfqpoint{0.75cm}{1.67cm}}{\pgfqpoint{0.775cm}{1.696cm}}
\pgfpathcurveto{\pgfqpoint{0.801cm}{1.722cm}}{\pgfqpoint{0.815cm}{1.756cm}}{\pgfqpoint{0.815cm}{1.793cm}}
\pgfusepath{fill}
\pgfpathmoveto{\pgfqpoint{1.345cm}{1.765cm}}
\pgfpathcurveto{\pgfqpoint{1.345cm}{1.801cm}}{\pgfqpoint{1.331cm}{1.836cm}}{\pgfqpoint{1.305cm}{1.862cm}}
\pgfpathcurveto{\pgfqpoint{1.28cm}{1.887cm}}{\pgfqpoint{1.245cm}{1.902cm}}{\pgfqpoint{1.209cm}{1.902cm}}
\pgfpathcurveto{\pgfqpoint{1.172cm}{1.902cm}}{\pgfqpoint{1.138cm}{1.887cm}}{\pgfqpoint{1.112cm}{1.862cm}}
\pgfpathcurveto{\pgfqpoint{1.087cm}{1.836cm}}{\pgfqpoint{1.072cm}{1.801cm}}{\pgfqpoint{1.072cm}{1.765cm}}
\pgfpathcurveto{\pgfqpoint{1.072cm}{1.728cm}}{\pgfqpoint{1.087cm}{1.694cm}}{\pgfqpoint{1.112cm}{1.668cm}}
\pgfpathcurveto{\pgfqpoint{1.138cm}{1.642cm}}{\pgfqpoint{1.172cm}{1.628cm}}{\pgfqpoint{1.209cm}{1.628cm}}
\pgfpathcurveto{\pgfqpoint{1.245cm}{1.628cm}}{\pgfqpoint{1.28cm}{1.642cm}}{\pgfqpoint{1.305cm}{1.668cm}}
\pgfpathcurveto{\pgfqpoint{1.331cm}{1.694cm}}{\pgfqpoint{1.345cm}{1.728cm}}{\pgfqpoint{1.345cm}{1.765cm}}
\pgfusepath{fill}
\begin{pgfscope}
\pgfsetdash{}{0cm}
\pgfsetlinewidth{0.818mm}
\pgfsetroundcap
\pgfsetroundjoin
\pgfsetmiterlimit{7.0}
\pgfpathmoveto{\pgfqpoint{0.682cm}{1.065cm}}
\pgfpathlineto{\pgfqpoint{1.246cm}{0.315cm}}
\pgfpathlineto{\pgfqpoint{1.811cm}{1.065cm}}
\pgfusepath{stroke}
\end{pgfscope}
\pgfpathmoveto{\pgfqpoint{1.948cm}{1.065cm}}
\pgfpathcurveto{\pgfqpoint{1.948cm}{1.101cm}}{\pgfqpoint{1.933cm}{1.136cm}}{\pgfqpoint{1.907cm}{1.162cm}}
\pgfpathcurveto{\pgfqpoint{1.882cm}{1.187cm}}{\pgfqpoint{1.847cm}{1.202cm}}{\pgfqpoint{1.811cm}{1.202cm}}
\pgfpathcurveto{\pgfqpoint{1.775cm}{1.202cm}}{\pgfqpoint{1.74cm}{1.187cm}}{\pgfqpoint{1.714cm}{1.162cm}}
\pgfpathcurveto{\pgfqpoint{1.689cm}{1.136cm}}{\pgfqpoint{1.674cm}{1.101cm}}{\pgfqpoint{1.674cm}{1.065cm}}
\pgfpathcurveto{\pgfqpoint{1.674cm}{1.029cm}}{\pgfqpoint{1.689cm}{0.994cm}}{\pgfqpoint{1.714cm}{0.968cm}}
\pgfpathcurveto{\pgfqpoint{1.74cm}{0.942cm}}{\pgfqpoint{1.775cm}{0.928cm}}{\pgfqpoint{1.811cm}{0.928cm}}
\pgfpathcurveto{\pgfqpoint{1.847cm}{0.928cm}}{\pgfqpoint{1.882cm}{0.942cm}}{\pgfqpoint{1.907cm}{0.968cm}}
\pgfpathcurveto{\pgfqpoint{1.933cm}{0.994cm}}{\pgfqpoint{1.948cm}{1.029cm}}{\pgfqpoint{1.948cm}{1.065cm}}
\pgfusepath{fill}
\begin{pgfscope}
\pgfsetdash{}{0cm}
\pgfsetlinewidth{0.818mm}
\pgfsetmiterlimit{4.0}
\pgfpathmoveto{\pgfqpoint{1.383cm}{0.178cm}}
\pgfpathcurveto{\pgfqpoint{1.383cm}{0.214cm}}{\pgfqpoint{1.369cm}{0.249cm}}{\pgfqpoint{1.343cm}{0.275cm}}
\pgfpathcurveto{\pgfqpoint{1.317cm}{0.3cm}}{\pgfqpoint{1.283cm}{0.315cm}}{\pgfqpoint{1.246cm}{0.315cm}}
\pgfpathcurveto{\pgfqpoint{1.21cm}{0.315cm}}{\pgfqpoint{1.175cm}{0.3cm}}{\pgfqpoint{1.15cm}{0.275cm}}
\pgfpathcurveto{\pgfqpoint{1.124cm}{0.249cm}}{\pgfqpoint{1.11cm}{0.214cm}}{\pgfqpoint{1.11cm}{0.178cm}}
\pgfpathcurveto{\pgfqpoint{1.11cm}{0.141cm}}{\pgfqpoint{1.124cm}{0.107cm}}{\pgfqpoint{1.15cm}{0.081cm}}
\pgfpathcurveto{\pgfqpoint{1.175cm}{0.055cm}}{\pgfqpoint{1.21cm}{0.041cm}}{\pgfqpoint{1.246cm}{0.041cm}}
\pgfpathcurveto{\pgfqpoint{1.283cm}{0.041cm}}{\pgfqpoint{1.317cm}{0.055cm}}{\pgfqpoint{1.343cm}{0.081cm}}
\pgfpathcurveto{\pgfqpoint{1.369cm}{0.107cm}}{\pgfqpoint{1.383cm}{0.141cm}}{\pgfqpoint{1.383cm}{0.178cm}}
\pgfusepath{stroke}
\end{pgfscope}
\end{pgfscope}
\end{pgfscope}
\end{pgfscope}
\end{tikzpicture}}}\|_{\CC^\gamma(\rho^{3+\gamma})}\lesssim\|\phi+\psi\|_{\CC^{\gamma+\kappa}(\rho^{2+\gamma+\kappa})}\lesssim 1+\|\psi\|_{\CC^{1+\alpha}(\rho^{2+\alpha})},
\end{align*}
\begin{align*}
\begin{aligned}
\|3\UU_\leqslant X\succ (\phi+\psi)^2\|_{\CC^\gamma(\rho^{3+\gamma})}
&\lesssim\|3\UU_\leqslant X\|_{\CC^\gamma(\rho^{1+\gamma})}\|\phi+\psi\|^2_{L^\infty(\rho)}\lesssim 2^{(\frac{1}{2}+\gamma+\kappa)\frac{4}{3}K}\|\phi+\psi\|^2_{L^\infty(\rho)}\\
&\lesssim 1+\|\psi\|_{L^\infty(\rho)}^{2+\varepsilon}
\end{aligned}
\end{align*}
and finally
\begin{align*}
\|3 X\preccurlyeq(\phi+\psi)^2\|_{\CC^\gamma(\rho^{3+\gamma})}
&\lesssim \|(\phi+\psi)^2\|_{\CC^{\frac{1}{2}+\gamma+\kappa}(\rho^{3})}\lesssim \|\phi+\psi\|_{L^\infty(\rho)}\|\phi+\psi\|_{\CC^{\frac{1}{2}+\gamma+\kappa}(\rho^{2})}\\
&\lesssim (1+\|\psi\|_{L^\infty(\rho)})(1+\|\psi\|_{L^\infty(\rho)}^\varepsilon+\|\psi\|_{\CC^{\frac{1}{2}+\alpha}(\rho^{\frac{3}{2}+\alpha})})\\
&\lesssim 1 +\|\psi\|_{L^\infty(\rho)}^{1+\varepsilon}+\|\psi\|_{\CC^{\frac{1}{2}+\alpha}(\rho^{\frac{3}{2}+\alpha})}+\|\psi\|_{L^\infty(\rho)}\|\psi\|_{\CC^{\frac{1}{2}+\alpha}(\rho^{\frac{3}{2}+\alpha})}.
\end{align*}

To summarize, we have shown that
\begin{align*}
\|\Psi\|_{\CC^{\gamma}(\rho^{3+\gamma})}
&\lesssim 1+ \|\psi\|_{\CC^{1+\alpha}(\rho^{2+\alpha})} +\|\psi\|_{\CC^{\frac{1}{2}+\alpha}(\rho^{\frac{3}{2}+\alpha})}
\\
&\quad +\|\psi\|^{2+\varepsilon}_{L^\infty(\rho)} 
+\|\psi\|^{1+\varepsilon}_{L^\infty(\rho)}\|\psi\|_{\CC^{\gamma}(\rho^{1+\gamma})}\\
&\quad +\|\psi\|_{L^\infty(\rho)}\|\psi\|_{\CC^{\frac{1}{2}+\alpha}(\rho^{\frac{3}{2}+\alpha})}.
\end{align*}
Due to interpolation from Lemma \ref{lemma:interp} we estimate
\begin{align*}
\|\psi\|_{\CC^{1+\alpha}(\rho^{2+\alpha})}\lesssim \|\psi\|^{1-\frac{1+\alpha}{2+\gamma}}_{L^\infty(\rho)} \|\psi\|^\frac{1+\alpha}{2+\gamma}_{\CC^{2+\gamma}(\rho^{3+\gamma})},
\end{align*}
\begin{align*}
\|\psi\|_{\CC^{\gamma}(\rho^{1+\gamma})}\lesssim \|\psi\|^{1-\frac{\gamma}{2+\gamma}}_{L^\infty(\rho)} \|\psi\|^\frac{\gamma}{2+\gamma}_{\CC^{2+\gamma}(\rho^{3+\gamma})},
\end{align*}
\begin{align*}
\|\psi\|_{\CC^{\frac{1}{2}+\alpha}(\rho^{\frac{3}{2}+\alpha})}\lesssim \|\psi\|^{1-\frac{\frac{1}{2}+\alpha}{2+\gamma}}_{L^\infty(\rho)} \|\psi\|^{\frac{\frac{1}{2}+\alpha}{2+\gamma}}_{\CC^{2+\gamma}(\rho^{3+\gamma})},
\end{align*}
therefore,  Lemma \ref{lemma:schauder-ellptic} together with the weighted Young inequality implies
\begin{align}\label{eq:97}
\|\psi\|_{\CC^{2+\gamma}(\rho^{3+\gamma})}\lesssim \|\Psi\|_{\CC^{\gamma}(\rho^{3+\gamma})}+\|\psi\|^{3+\gamma}_{L^\infty(\rho)} \lesssim 1 +\|\psi\|^{3+\gamma}_{L^\infty(\rho)} .
\end{align}

\subsection{Bound for $\psi$ in $L^\infty(\rho)$}
\label{ssec:psi2}

As the next step, towards the application of Lemma \ref{lemma:apriori-elliptic}, it is necessary to estimate $\Psi$ in $L^\infty(\rho^3)$. We observe that for most of the terms we may use the estimates above, only the cubic term is estimated as follows
\begin{align*}
&\|(-X^{\!\resizebox{0.6em}{!}{
\begin{tikzpicture}
\pgfpathmoveto{\pgfqpoint{0cm}{-0.035cm}}
\pgfpathlineto{\pgfqpoint{1.376cm}{-0.035cm}}
\pgfpathlineto{\pgfqpoint{1.376cm}{1.552cm}}
\pgfpathlineto{\pgfqpoint{0cm}{1.552cm}}
\pgfpathclose
\pgfusepath{clip}
\begin{pgfscope}
\begin{pgfscope}
\pgfpathmoveto{\pgfqpoint{0cm}{-0.035cm}}
\pgfpathlineto{\pgfqpoint{1.376cm}{-0.035cm}}
\pgfpathlineto{\pgfqpoint{1.376cm}{1.552cm}}
\pgfpathlineto{\pgfqpoint{0cm}{1.552cm}}
\pgfpathclose
\pgfusepath{clip}
\begin{pgfscope}
\begin{pgfscope}
\pgfsetdash{}{0cm}
\pgfsetlinewidth{0.818mm}
\pgfsetroundcap
\pgfsetroundjoin
\pgfsetmiterlimit{7.0}
\definecolor{eps2pgf_color}{gray}{0}\pgfsetstrokecolor{eps2pgf_color}\pgfsetfillcolor{eps2pgf_color}
\pgfpathmoveto{\pgfqpoint{0.117cm}{1.421cm}}
\pgfpathlineto{\pgfqpoint{0.682cm}{0.671cm}}
\pgfpathlineto{\pgfqpoint{1.246cm}{1.421cm}}
\pgfusepath{stroke}
\end{pgfscope}
\definecolor{eps2pgf_color}{gray}{0}\pgfsetstrokecolor{eps2pgf_color}\pgfsetfillcolor{eps2pgf_color}
\pgfpathmoveto{\pgfqpoint{0.273cm}{1.395cm}}
\pgfpathcurveto{\pgfqpoint{0.273cm}{1.432cm}}{\pgfqpoint{0.259cm}{1.467cm}}{\pgfqpoint{0.233cm}{1.492cm}}
\pgfpathcurveto{\pgfqpoint{0.207cm}{1.518cm}}{\pgfqpoint{0.173cm}{1.532cm}}{\pgfqpoint{0.137cm}{1.532cm}}
\pgfpathcurveto{\pgfqpoint{0.1cm}{1.532cm}}{\pgfqpoint{0.066cm}{1.518cm}}{\pgfqpoint{0.04cm}{1.492cm}}
\pgfpathcurveto{\pgfqpoint{0.014cm}{1.467cm}}{\pgfqpoint{0cm}{1.432cm}}{\pgfqpoint{0cm}{1.395cm}}
\pgfpathcurveto{\pgfqpoint{0cm}{1.359cm}}{\pgfqpoint{0.014cm}{1.324cm}}{\pgfqpoint{0.04cm}{1.299cm}}
\pgfpathcurveto{\pgfqpoint{0.066cm}{1.273cm}}{\pgfqpoint{0.1cm}{1.258cm}}{\pgfqpoint{0.137cm}{1.258cm}}
\pgfpathcurveto{\pgfqpoint{0.173cm}{1.258cm}}{\pgfqpoint{0.207cm}{1.273cm}}{\pgfqpoint{0.233cm}{1.299cm}}
\pgfpathcurveto{\pgfqpoint{0.259cm}{1.324cm}}{\pgfqpoint{0.273cm}{1.359cm}}{\pgfqpoint{0.273cm}{1.395cm}}
\pgfusepath{fill}
\begin{pgfscope}
\pgfsetdash{}{0cm}
\pgfsetlinewidth{0.818mm}
\pgfsetmiterlimit{7.0}
\pgfpathmoveto{\pgfqpoint{0.682cm}{0.671cm}}
\pgfpathlineto{\pgfqpoint{0.679cm}{1.418cm}}
\pgfusepath{stroke}
\end{pgfscope}
\pgfpathmoveto{\pgfqpoint{0.815cm}{1.399cm}}
\pgfpathcurveto{\pgfqpoint{0.815cm}{1.435cm}}{\pgfqpoint{0.801cm}{1.47cm}}{\pgfqpoint{0.775cm}{1.496cm}}
\pgfpathcurveto{\pgfqpoint{0.75cm}{1.521cm}}{\pgfqpoint{0.715cm}{1.536cm}}{\pgfqpoint{0.679cm}{1.536cm}}
\pgfpathcurveto{\pgfqpoint{0.643cm}{1.536cm}}{\pgfqpoint{0.608cm}{1.521cm}}{\pgfqpoint{0.582cm}{1.496cm}}
\pgfpathcurveto{\pgfqpoint{0.557cm}{1.47cm}}{\pgfqpoint{0.542cm}{1.435cm}}{\pgfqpoint{0.542cm}{1.399cm}}
\pgfpathcurveto{\pgfqpoint{0.542cm}{1.363cm}}{\pgfqpoint{0.557cm}{1.328cm}}{\pgfqpoint{0.582cm}{1.302cm}}
\pgfpathcurveto{\pgfqpoint{0.608cm}{1.276cm}}{\pgfqpoint{0.643cm}{1.262cm}}{\pgfqpoint{0.679cm}{1.262cm}}
\pgfpathcurveto{\pgfqpoint{0.715cm}{1.262cm}}{\pgfqpoint{0.75cm}{1.276cm}}{\pgfqpoint{0.775cm}{1.302cm}}
\pgfpathcurveto{\pgfqpoint{0.801cm}{1.328cm}}{\pgfqpoint{0.815cm}{1.363cm}}{\pgfqpoint{0.815cm}{1.399cm}}
\pgfusepath{fill}
\pgfpathmoveto{\pgfqpoint{1.345cm}{1.371cm}}
\pgfpathcurveto{\pgfqpoint{1.345cm}{1.408cm}}{\pgfqpoint{1.331cm}{1.442cm}}{\pgfqpoint{1.305cm}{1.468cm}}
\pgfpathcurveto{\pgfqpoint{1.28cm}{1.494cm}}{\pgfqpoint{1.245cm}{1.508cm}}{\pgfqpoint{1.209cm}{1.508cm}}
\pgfpathcurveto{\pgfqpoint{1.172cm}{1.508cm}}{\pgfqpoint{1.138cm}{1.494cm}}{\pgfqpoint{1.112cm}{1.468cm}}
\pgfpathcurveto{\pgfqpoint{1.087cm}{1.442cm}}{\pgfqpoint{1.072cm}{1.408cm}}{\pgfqpoint{1.072cm}{1.371cm}}
\pgfpathcurveto{\pgfqpoint{1.072cm}{1.335cm}}{\pgfqpoint{1.087cm}{1.3cm}}{\pgfqpoint{1.112cm}{1.274cm}}
\pgfpathcurveto{\pgfqpoint{1.138cm}{1.249cm}}{\pgfqpoint{1.172cm}{1.234cm}}{\pgfqpoint{1.209cm}{1.234cm}}
\pgfpathcurveto{\pgfqpoint{1.245cm}{1.234cm}}{\pgfqpoint{1.28cm}{1.249cm}}{\pgfqpoint{1.305cm}{1.274cm}}
\pgfpathcurveto{\pgfqpoint{1.331cm}{1.3cm}}{\pgfqpoint{1.345cm}{1.335cm}}{\pgfqpoint{1.345cm}{1.371cm}}
\pgfusepath{fill}
\begin{pgfscope}
\pgfsetdash{}{0cm}
\pgfsetlinewidth{0.818mm}
\pgfsetroundcap
\pgfsetmiterlimit{4.0}
\pgfpathmoveto{\pgfqpoint{0.682cm}{0.671cm}}
\pgfpathlineto{\pgfqpoint{0.682cm}{0.042cm}}
\pgfusepath{stroke}
\end{pgfscope}
\end{pgfscope}
\end{pgfscope}
\end{pgfscope}
\end{tikzpicture}}}+\phi)^3+3(-X^{\!\resizebox{0.6em}{!}{
\begin{tikzpicture}
\pgfpathmoveto{\pgfqpoint{0cm}{-0.035cm}}
\pgfpathlineto{\pgfqpoint{1.376cm}{-0.035cm}}
\pgfpathlineto{\pgfqpoint{1.376cm}{1.552cm}}
\pgfpathlineto{\pgfqpoint{0cm}{1.552cm}}
\pgfpathclose
\pgfusepath{clip}
\begin{pgfscope}
\begin{pgfscope}
\pgfpathmoveto{\pgfqpoint{0cm}{-0.035cm}}
\pgfpathlineto{\pgfqpoint{1.376cm}{-0.035cm}}
\pgfpathlineto{\pgfqpoint{1.376cm}{1.552cm}}
\pgfpathlineto{\pgfqpoint{0cm}{1.552cm}}
\pgfpathclose
\pgfusepath{clip}
\begin{pgfscope}
\begin{pgfscope}
\pgfsetdash{}{0cm}
\pgfsetlinewidth{0.818mm}
\pgfsetroundcap
\pgfsetroundjoin
\pgfsetmiterlimit{7.0}
\definecolor{eps2pgf_color}{gray}{0}\pgfsetstrokecolor{eps2pgf_color}\pgfsetfillcolor{eps2pgf_color}
\pgfpathmoveto{\pgfqpoint{0.117cm}{1.421cm}}
\pgfpathlineto{\pgfqpoint{0.682cm}{0.671cm}}
\pgfpathlineto{\pgfqpoint{1.246cm}{1.421cm}}
\pgfusepath{stroke}
\end{pgfscope}
\definecolor{eps2pgf_color}{gray}{0}\pgfsetstrokecolor{eps2pgf_color}\pgfsetfillcolor{eps2pgf_color}
\pgfpathmoveto{\pgfqpoint{0.273cm}{1.395cm}}
\pgfpathcurveto{\pgfqpoint{0.273cm}{1.432cm}}{\pgfqpoint{0.259cm}{1.467cm}}{\pgfqpoint{0.233cm}{1.492cm}}
\pgfpathcurveto{\pgfqpoint{0.207cm}{1.518cm}}{\pgfqpoint{0.173cm}{1.532cm}}{\pgfqpoint{0.137cm}{1.532cm}}
\pgfpathcurveto{\pgfqpoint{0.1cm}{1.532cm}}{\pgfqpoint{0.066cm}{1.518cm}}{\pgfqpoint{0.04cm}{1.492cm}}
\pgfpathcurveto{\pgfqpoint{0.014cm}{1.467cm}}{\pgfqpoint{0cm}{1.432cm}}{\pgfqpoint{0cm}{1.395cm}}
\pgfpathcurveto{\pgfqpoint{0cm}{1.359cm}}{\pgfqpoint{0.014cm}{1.324cm}}{\pgfqpoint{0.04cm}{1.299cm}}
\pgfpathcurveto{\pgfqpoint{0.066cm}{1.273cm}}{\pgfqpoint{0.1cm}{1.258cm}}{\pgfqpoint{0.137cm}{1.258cm}}
\pgfpathcurveto{\pgfqpoint{0.173cm}{1.258cm}}{\pgfqpoint{0.207cm}{1.273cm}}{\pgfqpoint{0.233cm}{1.299cm}}
\pgfpathcurveto{\pgfqpoint{0.259cm}{1.324cm}}{\pgfqpoint{0.273cm}{1.359cm}}{\pgfqpoint{0.273cm}{1.395cm}}
\pgfusepath{fill}
\begin{pgfscope}
\pgfsetdash{}{0cm}
\pgfsetlinewidth{0.818mm}
\pgfsetmiterlimit{7.0}
\pgfpathmoveto{\pgfqpoint{0.682cm}{0.671cm}}
\pgfpathlineto{\pgfqpoint{0.679cm}{1.418cm}}
\pgfusepath{stroke}
\end{pgfscope}
\pgfpathmoveto{\pgfqpoint{0.815cm}{1.399cm}}
\pgfpathcurveto{\pgfqpoint{0.815cm}{1.435cm}}{\pgfqpoint{0.801cm}{1.47cm}}{\pgfqpoint{0.775cm}{1.496cm}}
\pgfpathcurveto{\pgfqpoint{0.75cm}{1.521cm}}{\pgfqpoint{0.715cm}{1.536cm}}{\pgfqpoint{0.679cm}{1.536cm}}
\pgfpathcurveto{\pgfqpoint{0.643cm}{1.536cm}}{\pgfqpoint{0.608cm}{1.521cm}}{\pgfqpoint{0.582cm}{1.496cm}}
\pgfpathcurveto{\pgfqpoint{0.557cm}{1.47cm}}{\pgfqpoint{0.542cm}{1.435cm}}{\pgfqpoint{0.542cm}{1.399cm}}
\pgfpathcurveto{\pgfqpoint{0.542cm}{1.363cm}}{\pgfqpoint{0.557cm}{1.328cm}}{\pgfqpoint{0.582cm}{1.302cm}}
\pgfpathcurveto{\pgfqpoint{0.608cm}{1.276cm}}{\pgfqpoint{0.643cm}{1.262cm}}{\pgfqpoint{0.679cm}{1.262cm}}
\pgfpathcurveto{\pgfqpoint{0.715cm}{1.262cm}}{\pgfqpoint{0.75cm}{1.276cm}}{\pgfqpoint{0.775cm}{1.302cm}}
\pgfpathcurveto{\pgfqpoint{0.801cm}{1.328cm}}{\pgfqpoint{0.815cm}{1.363cm}}{\pgfqpoint{0.815cm}{1.399cm}}
\pgfusepath{fill}
\pgfpathmoveto{\pgfqpoint{1.345cm}{1.371cm}}
\pgfpathcurveto{\pgfqpoint{1.345cm}{1.408cm}}{\pgfqpoint{1.331cm}{1.442cm}}{\pgfqpoint{1.305cm}{1.468cm}}
\pgfpathcurveto{\pgfqpoint{1.28cm}{1.494cm}}{\pgfqpoint{1.245cm}{1.508cm}}{\pgfqpoint{1.209cm}{1.508cm}}
\pgfpathcurveto{\pgfqpoint{1.172cm}{1.508cm}}{\pgfqpoint{1.138cm}{1.494cm}}{\pgfqpoint{1.112cm}{1.468cm}}
\pgfpathcurveto{\pgfqpoint{1.087cm}{1.442cm}}{\pgfqpoint{1.072cm}{1.408cm}}{\pgfqpoint{1.072cm}{1.371cm}}
\pgfpathcurveto{\pgfqpoint{1.072cm}{1.335cm}}{\pgfqpoint{1.087cm}{1.3cm}}{\pgfqpoint{1.112cm}{1.274cm}}
\pgfpathcurveto{\pgfqpoint{1.138cm}{1.249cm}}{\pgfqpoint{1.172cm}{1.234cm}}{\pgfqpoint{1.209cm}{1.234cm}}
\pgfpathcurveto{\pgfqpoint{1.245cm}{1.234cm}}{\pgfqpoint{1.28cm}{1.249cm}}{\pgfqpoint{1.305cm}{1.274cm}}
\pgfpathcurveto{\pgfqpoint{1.331cm}{1.3cm}}{\pgfqpoint{1.345cm}{1.335cm}}{\pgfqpoint{1.345cm}{1.371cm}}
\pgfusepath{fill}
\begin{pgfscope}
\pgfsetdash{}{0cm}
\pgfsetlinewidth{0.818mm}
\pgfsetroundcap
\pgfsetmiterlimit{4.0}
\pgfpathmoveto{\pgfqpoint{0.682cm}{0.671cm}}
\pgfpathlineto{\pgfqpoint{0.682cm}{0.042cm}}
\pgfusepath{stroke}
\end{pgfscope}
\end{pgfscope}
\end{pgfscope}
\end{pgfscope}
\end{tikzpicture}}}+\phi)^2\psi+3(-X^{\!\resizebox{0.6em}{!}{
\begin{tikzpicture}
\pgfpathmoveto{\pgfqpoint{0cm}{-0.035cm}}
\pgfpathlineto{\pgfqpoint{1.376cm}{-0.035cm}}
\pgfpathlineto{\pgfqpoint{1.376cm}{1.552cm}}
\pgfpathlineto{\pgfqpoint{0cm}{1.552cm}}
\pgfpathclose
\pgfusepath{clip}
\begin{pgfscope}
\begin{pgfscope}
\pgfpathmoveto{\pgfqpoint{0cm}{-0.035cm}}
\pgfpathlineto{\pgfqpoint{1.376cm}{-0.035cm}}
\pgfpathlineto{\pgfqpoint{1.376cm}{1.552cm}}
\pgfpathlineto{\pgfqpoint{0cm}{1.552cm}}
\pgfpathclose
\pgfusepath{clip}
\begin{pgfscope}
\begin{pgfscope}
\pgfsetdash{}{0cm}
\pgfsetlinewidth{0.818mm}
\pgfsetroundcap
\pgfsetroundjoin
\pgfsetmiterlimit{7.0}
\definecolor{eps2pgf_color}{gray}{0}\pgfsetstrokecolor{eps2pgf_color}\pgfsetfillcolor{eps2pgf_color}
\pgfpathmoveto{\pgfqpoint{0.117cm}{1.421cm}}
\pgfpathlineto{\pgfqpoint{0.682cm}{0.671cm}}
\pgfpathlineto{\pgfqpoint{1.246cm}{1.421cm}}
\pgfusepath{stroke}
\end{pgfscope}
\definecolor{eps2pgf_color}{gray}{0}\pgfsetstrokecolor{eps2pgf_color}\pgfsetfillcolor{eps2pgf_color}
\pgfpathmoveto{\pgfqpoint{0.273cm}{1.395cm}}
\pgfpathcurveto{\pgfqpoint{0.273cm}{1.432cm}}{\pgfqpoint{0.259cm}{1.467cm}}{\pgfqpoint{0.233cm}{1.492cm}}
\pgfpathcurveto{\pgfqpoint{0.207cm}{1.518cm}}{\pgfqpoint{0.173cm}{1.532cm}}{\pgfqpoint{0.137cm}{1.532cm}}
\pgfpathcurveto{\pgfqpoint{0.1cm}{1.532cm}}{\pgfqpoint{0.066cm}{1.518cm}}{\pgfqpoint{0.04cm}{1.492cm}}
\pgfpathcurveto{\pgfqpoint{0.014cm}{1.467cm}}{\pgfqpoint{0cm}{1.432cm}}{\pgfqpoint{0cm}{1.395cm}}
\pgfpathcurveto{\pgfqpoint{0cm}{1.359cm}}{\pgfqpoint{0.014cm}{1.324cm}}{\pgfqpoint{0.04cm}{1.299cm}}
\pgfpathcurveto{\pgfqpoint{0.066cm}{1.273cm}}{\pgfqpoint{0.1cm}{1.258cm}}{\pgfqpoint{0.137cm}{1.258cm}}
\pgfpathcurveto{\pgfqpoint{0.173cm}{1.258cm}}{\pgfqpoint{0.207cm}{1.273cm}}{\pgfqpoint{0.233cm}{1.299cm}}
\pgfpathcurveto{\pgfqpoint{0.259cm}{1.324cm}}{\pgfqpoint{0.273cm}{1.359cm}}{\pgfqpoint{0.273cm}{1.395cm}}
\pgfusepath{fill}
\begin{pgfscope}
\pgfsetdash{}{0cm}
\pgfsetlinewidth{0.818mm}
\pgfsetmiterlimit{7.0}
\pgfpathmoveto{\pgfqpoint{0.682cm}{0.671cm}}
\pgfpathlineto{\pgfqpoint{0.679cm}{1.418cm}}
\pgfusepath{stroke}
\end{pgfscope}
\pgfpathmoveto{\pgfqpoint{0.815cm}{1.399cm}}
\pgfpathcurveto{\pgfqpoint{0.815cm}{1.435cm}}{\pgfqpoint{0.801cm}{1.47cm}}{\pgfqpoint{0.775cm}{1.496cm}}
\pgfpathcurveto{\pgfqpoint{0.75cm}{1.521cm}}{\pgfqpoint{0.715cm}{1.536cm}}{\pgfqpoint{0.679cm}{1.536cm}}
\pgfpathcurveto{\pgfqpoint{0.643cm}{1.536cm}}{\pgfqpoint{0.608cm}{1.521cm}}{\pgfqpoint{0.582cm}{1.496cm}}
\pgfpathcurveto{\pgfqpoint{0.557cm}{1.47cm}}{\pgfqpoint{0.542cm}{1.435cm}}{\pgfqpoint{0.542cm}{1.399cm}}
\pgfpathcurveto{\pgfqpoint{0.542cm}{1.363cm}}{\pgfqpoint{0.557cm}{1.328cm}}{\pgfqpoint{0.582cm}{1.302cm}}
\pgfpathcurveto{\pgfqpoint{0.608cm}{1.276cm}}{\pgfqpoint{0.643cm}{1.262cm}}{\pgfqpoint{0.679cm}{1.262cm}}
\pgfpathcurveto{\pgfqpoint{0.715cm}{1.262cm}}{\pgfqpoint{0.75cm}{1.276cm}}{\pgfqpoint{0.775cm}{1.302cm}}
\pgfpathcurveto{\pgfqpoint{0.801cm}{1.328cm}}{\pgfqpoint{0.815cm}{1.363cm}}{\pgfqpoint{0.815cm}{1.399cm}}
\pgfusepath{fill}
\pgfpathmoveto{\pgfqpoint{1.345cm}{1.371cm}}
\pgfpathcurveto{\pgfqpoint{1.345cm}{1.408cm}}{\pgfqpoint{1.331cm}{1.442cm}}{\pgfqpoint{1.305cm}{1.468cm}}
\pgfpathcurveto{\pgfqpoint{1.28cm}{1.494cm}}{\pgfqpoint{1.245cm}{1.508cm}}{\pgfqpoint{1.209cm}{1.508cm}}
\pgfpathcurveto{\pgfqpoint{1.172cm}{1.508cm}}{\pgfqpoint{1.138cm}{1.494cm}}{\pgfqpoint{1.112cm}{1.468cm}}
\pgfpathcurveto{\pgfqpoint{1.087cm}{1.442cm}}{\pgfqpoint{1.072cm}{1.408cm}}{\pgfqpoint{1.072cm}{1.371cm}}
\pgfpathcurveto{\pgfqpoint{1.072cm}{1.335cm}}{\pgfqpoint{1.087cm}{1.3cm}}{\pgfqpoint{1.112cm}{1.274cm}}
\pgfpathcurveto{\pgfqpoint{1.138cm}{1.249cm}}{\pgfqpoint{1.172cm}{1.234cm}}{\pgfqpoint{1.209cm}{1.234cm}}
\pgfpathcurveto{\pgfqpoint{1.245cm}{1.234cm}}{\pgfqpoint{1.28cm}{1.249cm}}{\pgfqpoint{1.305cm}{1.274cm}}
\pgfpathcurveto{\pgfqpoint{1.331cm}{1.3cm}}{\pgfqpoint{1.345cm}{1.335cm}}{\pgfqpoint{1.345cm}{1.371cm}}
\pgfusepath{fill}
\begin{pgfscope}
\pgfsetdash{}{0cm}
\pgfsetlinewidth{0.818mm}
\pgfsetroundcap
\pgfsetmiterlimit{4.0}
\pgfpathmoveto{\pgfqpoint{0.682cm}{0.671cm}}
\pgfpathlineto{\pgfqpoint{0.682cm}{0.042cm}}
\pgfusepath{stroke}
\end{pgfscope}
\end{pgfscope}
\end{pgfscope}
\end{pgfscope}
\end{tikzpicture}}}+\phi)\psi^2\|_{L^\infty(\rho^{3})}\lesssim 1+\|\psi\|^2_{L^\infty(\rho)},
\end{align*}
and we may also improve the bound
\begin{align*}
\begin{aligned}
\|3\UU_\leqslant\llbracket X^2 \rrbracket\prec(\phi+\psi)\|_{L^\infty(\rho^{3})}&\lesssim\|\phi+\psi\|_{L^\infty(\rho)} \|\UU_\leqslant\llbracket X^2 \rrbracket\|_{\CC^\gamma(\rho^{2})}\\
&\lesssim 2^{(1+\gamma+\kappa)K}\|\phi+\psi\|_{L^\infty(\rho)} \lesssim 1+ \|\psi\|^{2+\varepsilon}_{L^\infty(\rho)}.
\end{aligned}
\end{align*}
Therefore, we deduce
\begin{align*}
\|\Psi\|_{L^\infty(\rho^{3})}
&\lesssim 1+ \|\psi\|_{\CC^{1+\alpha}(\rho^{2+\alpha})} +\|\psi\|_{\CC^{\frac{1}{2}+\alpha}(\rho^{\frac{3}{2}+\alpha})}
\\
&\quad +\|\psi\|^{2+\varepsilon}_{L^\infty(\rho)} 
 +\|\psi\|_{L^\infty(\rho)}\|\psi\|_{\CC^{\frac{1}{2}+\alpha}(\rho^{\frac{3}{2}+\alpha})}
\end{align*}
and applying again the interpolation from Lemma \ref{lemma:interp} together with \eqref{eq:97} leads to
\begin{align*}
\|\Psi\|_{L^\infty(\rho^{3})}
&\lesssim 1+ \|\psi\|_{L^\infty(\rho)}^{2+\varepsilon} .
\end{align*}
Finally, according to Lemma \ref{lemma:apriori-elliptic} and weighted Young inequality we conclude
$$
\|\psi\|_{L^\infty(\rho)}\lesssim 1+ \|\Psi\|_{L^\infty(\rho)}^{1/3}\lesssim 1 + \|\psi\|_{L^\infty(\rho)}^\varepsilon\lesssim 1,
$$
and the proof is complete.

\subsection{Existence}
\label{ssec:ex45}

The construction of a solution proceeds similarly to Section~\ref{ssec:ex44}. More precisely,  we first consider the problem on a large torus of size $M$ and establish existence based on Schaefer's fixed point theorem \cite[Section 9.2.2, Theorem 4]{E98}. Then we make use of the a priori estimates from Sections~\ref{ssec:phi1}, \ref{ssec:phi2}, \ref{ssec:theta}, \ref{ssec:psi2}, \ref{ssec:psi2} together with Theorem \ref{thm:renorm2} and a compactness argument to pass to the limit as $M\to\infty$.

Recall that in view of the computations in Sections \ref{ssec:dec45}, \ref{s:phi}, system \eqref{eq:phi4} in dimension~5 reduces to equations \eqref{eq:two}, \eqref{eq:th}.

\begin{theorem}\label{thm:ex43}
Let $\kappa, \alpha\in (0,1)$ be chosen sufficiently small and let $\gamma=\alpha-\kappa>0$. There exists
$$(\phi,\psi)\in  [ \CC^{\frac 12+\alpha}(\rho^{\frac 32+\alpha})\cap \CC^{\alpha}(\rho)]\times[\CC^{2+\gamma}(\rho^{3+\gamma})\cap L^{\infty}(\rho)]$$
which is a solution to \eqref{eq:two}, \eqref{eq:th} on $\R^{5}$.
\end{theorem}

\begin{proof} \emph{Step 1 -- existence on a large torus:}
Similarly to the proof of Theorem \ref{thm:ex-44-per}, we define a fixed point map
$$\mathcal{K}:\CC^{\frac 12+\beta}(\mathbb{T}_{M}^{5})\times \CC^{1+\beta}(\mathbb{T}_{M}^{5})\to\CC^{\frac 12+\beta}(\mathbb{T}_{M}^{5})\times \CC^{1+\beta}(\mathbb{T}_{M}^{5})$$ for a small parameter $\beta\in(0,1)$ as follows: given
$$(\tilde\phi,\tilde\psi)\in \CC^{\frac 12+\beta}(\mathbb{T}_{M}^{5})\times \CC^{ 1+\beta}(\mathbb{T}_{M}^{5}),$$
let $\mathcal{K}(\tilde\phi,\tilde\psi)=(\phi,\psi)$ be a solution to 
\begin{equation}\label{eq:45d}
\Q \phi + \Phi(\tilde\phi,\tilde\psi) = 0, \qquad \Q \psi + \psi^3 + \Psi(\tilde\phi,\tilde\psi) = 0,
\end{equation}
where $\Phi(\tilde\phi,\tilde\psi)$ and $\Psi(\tilde\phi,\tilde\psi)$ contain all the magenta and blue terms from Section \ref{s:phi}, respectively, with $\phi,\psi$ replaced by $\tilde\phi,\tilde
\psi$.

The first equation in \eqref{eq:45d} always has a (unique) solution $\phi$ which belongs to $\CC^{\alpha}(\mathbb{T}_{M}^{5})$  due to the bounds in Section \ref{ssec:phi1}. Moreover, Section \ref{ssec:phi2} shows that $\phi\in\CC^{\frac 12+\alpha}(\mathbb{T}_{M}^{5})$ and we may choose $\alpha>\beta$. Furthermore, similarly to \eqref{eq:th} we denote
$$
\vartheta:=\phi +3(-X^{\!\resizebox{0.6em}{!}{
\begin{tikzpicture}
\pgfpathmoveto{\pgfqpoint{0cm}{-0.035cm}}
\pgfpathlineto{\pgfqpoint{1.376cm}{-0.035cm}}
\pgfpathlineto{\pgfqpoint{1.376cm}{1.552cm}}
\pgfpathlineto{\pgfqpoint{0cm}{1.552cm}}
\pgfpathclose
\pgfusepath{clip}
\begin{pgfscope}
\begin{pgfscope}
\pgfpathmoveto{\pgfqpoint{0cm}{-0.035cm}}
\pgfpathlineto{\pgfqpoint{1.376cm}{-0.035cm}}
\pgfpathlineto{\pgfqpoint{1.376cm}{1.552cm}}
\pgfpathlineto{\pgfqpoint{0cm}{1.552cm}}
\pgfpathclose
\pgfusepath{clip}
\begin{pgfscope}
\begin{pgfscope}
\pgfsetdash{}{0cm}
\pgfsetlinewidth{0.818mm}
\pgfsetroundcap
\pgfsetroundjoin
\pgfsetmiterlimit{7.0}
\definecolor{eps2pgf_color}{gray}{0}\pgfsetstrokecolor{eps2pgf_color}\pgfsetfillcolor{eps2pgf_color}
\pgfpathmoveto{\pgfqpoint{0.117cm}{1.421cm}}
\pgfpathlineto{\pgfqpoint{0.682cm}{0.671cm}}
\pgfpathlineto{\pgfqpoint{1.246cm}{1.421cm}}
\pgfusepath{stroke}
\end{pgfscope}
\definecolor{eps2pgf_color}{gray}{0}\pgfsetstrokecolor{eps2pgf_color}\pgfsetfillcolor{eps2pgf_color}
\pgfpathmoveto{\pgfqpoint{0.273cm}{1.395cm}}
\pgfpathcurveto{\pgfqpoint{0.273cm}{1.432cm}}{\pgfqpoint{0.259cm}{1.467cm}}{\pgfqpoint{0.233cm}{1.492cm}}
\pgfpathcurveto{\pgfqpoint{0.207cm}{1.518cm}}{\pgfqpoint{0.173cm}{1.532cm}}{\pgfqpoint{0.137cm}{1.532cm}}
\pgfpathcurveto{\pgfqpoint{0.1cm}{1.532cm}}{\pgfqpoint{0.066cm}{1.518cm}}{\pgfqpoint{0.04cm}{1.492cm}}
\pgfpathcurveto{\pgfqpoint{0.014cm}{1.467cm}}{\pgfqpoint{0cm}{1.432cm}}{\pgfqpoint{0cm}{1.395cm}}
\pgfpathcurveto{\pgfqpoint{0cm}{1.359cm}}{\pgfqpoint{0.014cm}{1.324cm}}{\pgfqpoint{0.04cm}{1.299cm}}
\pgfpathcurveto{\pgfqpoint{0.066cm}{1.273cm}}{\pgfqpoint{0.1cm}{1.258cm}}{\pgfqpoint{0.137cm}{1.258cm}}
\pgfpathcurveto{\pgfqpoint{0.173cm}{1.258cm}}{\pgfqpoint{0.207cm}{1.273cm}}{\pgfqpoint{0.233cm}{1.299cm}}
\pgfpathcurveto{\pgfqpoint{0.259cm}{1.324cm}}{\pgfqpoint{0.273cm}{1.359cm}}{\pgfqpoint{0.273cm}{1.395cm}}
\pgfusepath{fill}
\begin{pgfscope}
\pgfsetdash{}{0cm}
\pgfsetlinewidth{0.818mm}
\pgfsetmiterlimit{7.0}
\pgfpathmoveto{\pgfqpoint{0.682cm}{0.671cm}}
\pgfpathlineto{\pgfqpoint{0.679cm}{1.418cm}}
\pgfusepath{stroke}
\end{pgfscope}
\pgfpathmoveto{\pgfqpoint{0.815cm}{1.399cm}}
\pgfpathcurveto{\pgfqpoint{0.815cm}{1.435cm}}{\pgfqpoint{0.801cm}{1.47cm}}{\pgfqpoint{0.775cm}{1.496cm}}
\pgfpathcurveto{\pgfqpoint{0.75cm}{1.521cm}}{\pgfqpoint{0.715cm}{1.536cm}}{\pgfqpoint{0.679cm}{1.536cm}}
\pgfpathcurveto{\pgfqpoint{0.643cm}{1.536cm}}{\pgfqpoint{0.608cm}{1.521cm}}{\pgfqpoint{0.582cm}{1.496cm}}
\pgfpathcurveto{\pgfqpoint{0.557cm}{1.47cm}}{\pgfqpoint{0.542cm}{1.435cm}}{\pgfqpoint{0.542cm}{1.399cm}}
\pgfpathcurveto{\pgfqpoint{0.542cm}{1.363cm}}{\pgfqpoint{0.557cm}{1.328cm}}{\pgfqpoint{0.582cm}{1.302cm}}
\pgfpathcurveto{\pgfqpoint{0.608cm}{1.276cm}}{\pgfqpoint{0.643cm}{1.262cm}}{\pgfqpoint{0.679cm}{1.262cm}}
\pgfpathcurveto{\pgfqpoint{0.715cm}{1.262cm}}{\pgfqpoint{0.75cm}{1.276cm}}{\pgfqpoint{0.775cm}{1.302cm}}
\pgfpathcurveto{\pgfqpoint{0.801cm}{1.328cm}}{\pgfqpoint{0.815cm}{1.363cm}}{\pgfqpoint{0.815cm}{1.399cm}}
\pgfusepath{fill}
\pgfpathmoveto{\pgfqpoint{1.345cm}{1.371cm}}
\pgfpathcurveto{\pgfqpoint{1.345cm}{1.408cm}}{\pgfqpoint{1.331cm}{1.442cm}}{\pgfqpoint{1.305cm}{1.468cm}}
\pgfpathcurveto{\pgfqpoint{1.28cm}{1.494cm}}{\pgfqpoint{1.245cm}{1.508cm}}{\pgfqpoint{1.209cm}{1.508cm}}
\pgfpathcurveto{\pgfqpoint{1.172cm}{1.508cm}}{\pgfqpoint{1.138cm}{1.494cm}}{\pgfqpoint{1.112cm}{1.468cm}}
\pgfpathcurveto{\pgfqpoint{1.087cm}{1.442cm}}{\pgfqpoint{1.072cm}{1.408cm}}{\pgfqpoint{1.072cm}{1.371cm}}
\pgfpathcurveto{\pgfqpoint{1.072cm}{1.335cm}}{\pgfqpoint{1.087cm}{1.3cm}}{\pgfqpoint{1.112cm}{1.274cm}}
\pgfpathcurveto{\pgfqpoint{1.138cm}{1.249cm}}{\pgfqpoint{1.172cm}{1.234cm}}{\pgfqpoint{1.209cm}{1.234cm}}
\pgfpathcurveto{\pgfqpoint{1.245cm}{1.234cm}}{\pgfqpoint{1.28cm}{1.249cm}}{\pgfqpoint{1.305cm}{1.274cm}}
\pgfpathcurveto{\pgfqpoint{1.331cm}{1.3cm}}{\pgfqpoint{1.345cm}{1.335cm}}{\pgfqpoint{1.345cm}{1.371cm}}
\pgfusepath{fill}
\begin{pgfscope}
\pgfsetdash{}{0cm}
\pgfsetlinewidth{0.818mm}
\pgfsetroundcap
\pgfsetmiterlimit{4.0}
\pgfpathmoveto{\pgfqpoint{0.682cm}{0.671cm}}
\pgfpathlineto{\pgfqpoint{0.682cm}{0.042cm}}
\pgfusepath{stroke}
\end{pgfscope}
\end{pgfscope}
\end{pgfscope}
\end{pgfscope}
\end{tikzpicture}}} + \tilde\phi + \tilde\psi)\prec X^{\!\resizebox{0.6em}{!}{
\begin{tikzpicture}
\pgfpathmoveto{\pgfqpoint{0cm}{0cm}}
\pgfpathlineto{\pgfqpoint{1.376cm}{0cm}}
\pgfpathlineto{\pgfqpoint{1.376cm}{1.588cm}}
\pgfpathlineto{\pgfqpoint{0cm}{1.588cm}}
\pgfpathclose
\pgfusepath{clip}
\begin{pgfscope}
\begin{pgfscope}
\pgfpathmoveto{\pgfqpoint{0cm}{0cm}}
\pgfpathlineto{\pgfqpoint{1.376cm}{0cm}}
\pgfpathlineto{\pgfqpoint{1.376cm}{1.588cm}}
\pgfpathlineto{\pgfqpoint{0cm}{1.588cm}}
\pgfpathclose
\pgfusepath{clip}
\begin{pgfscope}
\begin{pgfscope}
\definecolor{eps2pgf_color}{gray}{0.976471}\pgfsetstrokecolor{eps2pgf_color}\pgfsetfillcolor{eps2pgf_color}
\pgfpathmoveto{\pgfqpoint{0cm}{0cm}}
\pgfpathlineto{\pgfqpoint{1.376cm}{0cm}}
\pgfpathlineto{\pgfqpoint{1.376cm}{1.588cm}}
\pgfpathlineto{\pgfqpoint{0cm}{1.588cm}}
\pgfpathclose
\pgfusepath{fill}
\end{pgfscope}
\begin{pgfscope}
\pgfsetdash{}{0cm}
\pgfsetlinewidth{0.818mm}
\pgfsetroundcap
\pgfsetroundjoin
\pgfsetmiterlimit{7.0}
\definecolor{eps2pgf_color}{gray}{0}\pgfsetstrokecolor{eps2pgf_color}\pgfsetfillcolor{eps2pgf_color}
\pgfpathmoveto{\pgfqpoint{0.117cm}{1.476cm}}
\pgfpathlineto{\pgfqpoint{0.682cm}{0.726cm}}
\pgfpathlineto{\pgfqpoint{1.246cm}{1.476cm}}
\pgfusepath{stroke}
\end{pgfscope}
\definecolor{eps2pgf_color}{gray}{0}\pgfsetstrokecolor{eps2pgf_color}\pgfsetfillcolor{eps2pgf_color}
\pgfpathmoveto{\pgfqpoint{0.273cm}{1.451cm}}
\pgfpathcurveto{\pgfqpoint{0.273cm}{1.487cm}}{\pgfqpoint{0.259cm}{1.522cm}}{\pgfqpoint{0.233cm}{1.547cm}}
\pgfpathcurveto{\pgfqpoint{0.207cm}{1.573cm}}{\pgfqpoint{0.173cm}{1.588cm}}{\pgfqpoint{0.137cm}{1.588cm}}
\pgfpathcurveto{\pgfqpoint{0.1cm}{1.588cm}}{\pgfqpoint{0.066cm}{1.573cm}}{\pgfqpoint{0.04cm}{1.547cm}}
\pgfpathcurveto{\pgfqpoint{0.014cm}{1.522cm}}{\pgfqpoint{0cm}{1.487cm}}{\pgfqpoint{0cm}{1.451cm}}
\pgfpathcurveto{\pgfqpoint{0cm}{1.414cm}}{\pgfqpoint{0.014cm}{1.379cm}}{\pgfqpoint{0.04cm}{1.354cm}}
\pgfpathcurveto{\pgfqpoint{0.066cm}{1.328cm}}{\pgfqpoint{0.1cm}{1.314cm}}{\pgfqpoint{0.137cm}{1.314cm}}
\pgfpathcurveto{\pgfqpoint{0.173cm}{1.314cm}}{\pgfqpoint{0.207cm}{1.328cm}}{\pgfqpoint{0.233cm}{1.354cm}}
\pgfpathcurveto{\pgfqpoint{0.259cm}{1.379cm}}{\pgfqpoint{0.273cm}{1.414cm}}{\pgfqpoint{0.273cm}{1.451cm}}
\pgfusepath{fill}
\pgfpathmoveto{\pgfqpoint{1.345cm}{1.426cm}}
\pgfpathcurveto{\pgfqpoint{1.345cm}{1.463cm}}{\pgfqpoint{1.331cm}{1.497cm}}{\pgfqpoint{1.305cm}{1.523cm}}
\pgfpathcurveto{\pgfqpoint{1.28cm}{1.549cm}}{\pgfqpoint{1.245cm}{1.563cm}}{\pgfqpoint{1.209cm}{1.563cm}}
\pgfpathcurveto{\pgfqpoint{1.172cm}{1.563cm}}{\pgfqpoint{1.138cm}{1.549cm}}{\pgfqpoint{1.112cm}{1.523cm}}
\pgfpathcurveto{\pgfqpoint{1.087cm}{1.497cm}}{\pgfqpoint{1.072cm}{1.463cm}}{\pgfqpoint{1.072cm}{1.426cm}}
\pgfpathcurveto{\pgfqpoint{1.072cm}{1.39cm}}{\pgfqpoint{1.087cm}{1.355cm}}{\pgfqpoint{1.112cm}{1.329cm}}
\pgfpathcurveto{\pgfqpoint{1.138cm}{1.304cm}}{\pgfqpoint{1.172cm}{1.289cm}}{\pgfqpoint{1.209cm}{1.289cm}}
\pgfpathcurveto{\pgfqpoint{1.245cm}{1.289cm}}{\pgfqpoint{1.28cm}{1.304cm}}{\pgfqpoint{1.305cm}{1.329cm}}
\pgfpathcurveto{\pgfqpoint{1.331cm}{1.355cm}}{\pgfqpoint{1.345cm}{1.39cm}}{\pgfqpoint{1.345cm}{1.426cm}}
\pgfusepath{fill}
\begin{pgfscope}
\pgfsetdash{}{0cm}
\pgfsetlinewidth{0.818mm}
\pgfsetroundcap
\pgfsetmiterlimit{4.0}
\pgfpathmoveto{\pgfqpoint{0.682cm}{0.726cm}}
\pgfpathlineto{\pgfqpoint{0.682cm}{0.097cm}}
\pgfusepath{stroke}
\end{pgfscope}
\end{pgfscope}
\end{pgfscope}
\end{pgfscope}
\end{tikzpicture}}}
$$
and observe that due to Section \ref{ssec:theta} and Section \ref{ssec:psi1} (performed on $\mathbb{T}^{5}_{M}$) the right hand side $\Psi(\tilde\phi,\tilde\psi)$ belongs to $\CC^{\gamma}(\mathbb{T}^{5}_{M})$ provided   $(\tilde\phi,\tilde\psi)\in \CC^{\frac 12+\beta}(\mathbb{T}^{5}_{M})\times \CC^{1+\beta} (\mathbb{T}^{5}_{M})$ and $\gamma=\beta-\kappa$, $\gamma\leq\frac 12-3\kappa$. Hence Proposition \ref{prop:aux-1} implies  existence of a unique classical solution to the second equation in \eqref{eq:45d}. Hence the map $\mathcal{K}$ is well-defined.

Next, we deduce that  the map $\mathcal{K}$ has a fixed point $(\phi,\psi)\in \CC^{\frac 12+\alpha} (\mathbb{T}_{M}^{5})\times \CC^{2+\gamma}(\mathbb{T}_{M}^{5})$ for $\alpha=\beta+\kappa$ and $\gamma=\beta-\kappa$. More precisely, the proof follows the lines of Proposition \ref{prop:fixp} and employs the estimates from Sections \ref{ssec:phi1}, \ref{ssec:phi2},  \ref{ssec:theta}, \ref{ssec:psi1}, \ref{ssec:psi2}. The proof of existence on $\mathbb{T}^{5}_{M}$ is therefore complete.

 \emph{Step 2 -- existence on the full space:} For $M\in\N$ let  $(\phi_{M},\psi_{M})$ denote the solution to \eqref{eq:two}, \eqref{eq:th} on $\mathbb{T}^{5}_{M}$  constructed above.
Then  the a priori estimates from Sections \ref{ssec:phi1}, \ref{ssec:phi2},  \ref{ssec:theta}, \ref{ssec:psi1}, \ref{ssec:psi2} apply and, in view of Theorem~\ref{thm:renorm2}, we conclude that the approximate solutions $(\phi_{M},\psi_{M})$ are bounded uniformly in $M$ in
$$[ \CC^{\frac 12+\alpha}(\rho^{\frac 32+\alpha})\cap \CC^{\alpha}(\rho)]\times[\CC^{2+\gamma}(\rho^{3+\gamma})\cap L^{\infty}(\rho)]$$
whenever $\rho$ is a polynomial bound. Due to \eqref{eq:emb}, this space is compactly embedded into $\CC^{\frac 12+\alpha'}(\rho^{\frac 32+\alpha''})\times\CC^{2+\gamma'}(\rho^{3+\gamma''})$ provided $\alpha'<\alpha<\alpha''$ and $\gamma'<\gamma<\gamma''$. Therefore, there exists a subsequence, still denoted $(\phi_{M},\psi_{M})$ which converges in $\CC^{\frac 12+\alpha'}(\rho^{\frac 32+\alpha''})\times\CC^{2+\gamma'}(\rho^{3+\gamma''})$ to certain
$$(\phi,\psi)\in  [ \CC^{\frac 12+\alpha}(\rho^{\frac 32+\alpha})\cap \CC^{\alpha}(\rho)]\times[\CC^{2+\gamma}(\rho^{3+\gamma})\cap L^{\infty}(\rho)].$$
Passing to the limit in  \eqref{eq:two}, \eqref{eq:th} concludes the proof of existence on the full space.
\end{proof}

\section{Parabolic $\Phi^4_2$ model}
\label{sec:42}

The analysis of the parabolic $\Phi^{4}$ model on $\R^{2}$, that is,
\begin{equation}\label{eq:42}
\partial_{t}\varphi+ (- \Delta + \mu) \varphi + \varphi^3 - 3 a \varphi - \xi = 0,\qquad\varphi(0)=\varphi_{0},
\end{equation}
where $\xi$ is a space-time white noise, is very similar to the elliptic $\Phi^{4}_{4}$ model on $\R^{4}$. Indeed, the regularity of the space white noise in dimension $4$ is the same as the regularity of the space-time white noise in dimension $2$. Without loss of generality we assume that the mass $\mu$ is strictly positive (otherwise we add a linear term with positive mass to both sides of \eqref{eq:42} and consider the original massive term as a right hand side, see Remark \ref{rem:mu}). Then we proceed as in Section~\ref{sec:44}, let $(- \Delta +
\mu) = \Q$ and $\LL=\partial_{t}+\Q\hspace{-.2em}$ and introduce the ansatz
\[ \varphi = X + \phi + \psi, \]
with
\begin{equation*}
\LL X = \xi, \qquad \llbracket X^3 \rrbracket \assign X^3 - 3 a X, \quad
   \llbracket X^2 \rrbracket \assign X^2 - a.
\end{equation*}
Recall that $X$ is chosen stationary.
This leads us to the system of equations
\begin{equation}\label{eq:42b}
\LL \phi + \Phi = 0, \quad \phi(0)=\phi_{0}=\varphi_{0},\qquad \LL \psi + \psi^3 + \Psi = 0,\quad \psi(0)=0,
\end{equation}
where $\varphi_{0}\in \CC^{\alpha}(\rho_{0})$; and  $\Phi$ and $\Psi$ are given as in \eqref{eq:44c} but employing the parabolic localizers $\VV_{>},\VV_{\leq}$ instead of $\UU_{>},\UU_{\leq}$.

The existence of a solution can now be proved by choosing a smooth and space periodic approximation $\xi_{\varepsilon}$ of the driving space-time white noise $\xi$, defined on the torus of size $\frac 1\varepsilon$ and solving \eqref{eq:42} on the approximate level with the associated renormalization constant $a_{\varepsilon}$. Subsequently, we may pass to the limit using  the above uniform estimates together with  compactness. To be more precise, let $\xi_{\varepsilon}$ be a periodic version of a {space-time mollification} of $\xi$ defined on $[0,\8)\times\mathbb{T}^{2}_{1/\varepsilon}$ and define $X_{\varepsilon}$ as the stationary solution to
\begin{equation*}
\LL X_{\varepsilon} = \xi_{\varepsilon}.
\end{equation*}
The other stochastic objects were defined in Theorem \ref{thm:renorm42}. Throughout this section, $\rho$ denotes   a polynomial space-time weight.

Let $\varphi_{\varepsilon,0}$ be a mollification of the initial condition $\varphi_{0}$. Then according to Proposition \ref{prop:42d}, for every $\varepsilon\in (0,1)$ there exists $\varphi_{\varepsilon}\in C^{\infty}([0,\8)\times \mathbb{T}^{2}_{1/\varepsilon})$ which is the unique classical solution to
\begin{equation*}
\LL \varphi_{\varepsilon} + \varphi_{\varepsilon}^3 - 3 a_{\varepsilon} \varphi_{\varepsilon} - \xi_{\varepsilon} = 0,\qquad\varphi(0)=\varphi_{\varepsilon,0}.
\end{equation*}
As the next step, we proceed with the same decomposition $\varphi_{\varepsilon}=X_{\varepsilon}+\phi_{\varepsilon}+\psi_{\varepsilon}$ as above, only starting from the mollified version $\xi_{\varepsilon}$ instead of from $\xi$. 
According to Corollary \ref{cor:42d} it holds for every $\varepsilon\in (0,1)$ that $\varphi_{\varepsilon}\in  C\CC^{2 + \kappa} (\rho^{3 + \kappa})\cap C^{1}L^{\infty}(\rho^{3+\kappa}) \cap L^{\infty}L^{\infty} (\rho)$ and the  same regularity holds for $\phi_{\varepsilon},\psi_{\varepsilon}$. Hence we follow the lines of  Section~\ref{ssec:apr} and employ  Lemmas \ref{lemma:interp2}, \ref{lem:local2}, \ref{lemma:sch}, \ref{lemma:schauder-par}, \ref{lemma:apriori-parabolic}, in order to deduce that the following bound holds true uniformly in $\varepsilon\in (0,1)$
\begin{equation}\label{eq:un42}
\|\phi_{\varepsilon}\|_{C\CC^{\alpha}(\rho)}+\|\phi_{\varepsilon}\|_{C^{\alpha/2}L^{\infty}(\rho)}+\|\psi_{\varepsilon}\|_{C\CC^{2+\beta}(\rho^{3+\beta})}+\|\psi_{\varepsilon}\|_{C^{1}L^{\infty}(\rho^{3+\beta})}+\|\psi_{\varepsilon}\|_{L^{\infty}L^{\infty}(\rho)}\lesssim 1.
\end{equation}
Based on this uniform bound we are able to pass to the limit.

\begin{theorem}\label{thm:ex42}
Let $\kappa, \alpha\in (0,1)$ be chosen sufficiently small and let $\beta=\alpha-\kappa>0$. If $\varphi_{0}\in \CC^{\alpha}(\rho_{0})$ then there exists
$$(\phi,\psi)\in [C\CC^{\alpha}(\rho)\cap C^{\alpha/2}L^{\infty}(\rho)] \times[C\CC^{2+\beta}(\rho^{3+\beta})\cap C^{1}L^{\infty}(\rho^{3+\beta})\cap L^{\infty}L^{\infty}(\rho)]$$
which is a solution to  \eqref{eq:42b}.
\end{theorem}

\begin{proof}
Due to \eqref{eq:un42}, the approximate solutions $(\phi_{\varepsilon},\psi_{\varepsilon})$ are bounded uniformly in $\varepsilon
\in (0,1)$ in
$$[C\CC^{\alpha}(\rho)\cap C^{\alpha/2}L^{\infty}(\rho)] \times[C\CC^{2+\beta}(\rho^{3+\beta})\cap C^{1}L^{\infty}(\rho^{3+\beta})\cap L^{\infty}L^{\infty}(\rho)]$$
whenever $\rho$ is a polynomial bound. Due to \eqref{eq:emb}, Arzel\`a-Ascoli and Aubin-Lions-type argument (see \cite[Lemma 1, Theorem 5]{S87}) this space is compactly embedded into
$$[C_{\rm{loc}}\CC^{\alpha-\iota}(\rho^{1+\delta})\cap C^{(\alpha-\iota)/2}_{\rm{loc}}\CC^{-\gamma}(\rho^{1+\delta}) ]\times [C_{\rm{loc}}\CC^{2+\beta-\iota}(\rho^{3+\beta+\delta})\cap C^{1-\iota}_{\rm{loc}}\CC^{-\gamma}(\rho^{3+\beta+\delta}) ]$$ provided  $\iota\in (0,\alpha\wedge\beta\wedge1)$ and $\gamma,\delta\in (0,1)$ are chosen small. Therefore, there exists a subsequence, still denoted $(\phi_{\varepsilon},\psi_{\varepsilon})$ which converges in this space to certain $(\phi,\psi)$ and we intend to pass to the limit in  \eqref{eq:42b}.

To this end, we fix $T>0$. Note that due to Theorem~\ref{thm:renorm42}, the linearity of the localizers $\VV_{>},\VV_{\leq}$ and Lemma \ref{lem:local2} it follows that
\begin{align*}
\VV_{>} X_{\varepsilon}\to \VV_{>} X\quad\text{in}\quad C_{T}\CC^{\alpha-2} (\rho^{- 1}),
   &\qquad \VV_{\leqslant} X_{\varepsilon} \to \VV_{\leqslant} X\quad\text{in}\quad C_{T}\CC^{\alpha} (\rho^{1}),\\
 \VV_{>} \llbracket X_{\varepsilon}^2 \rrbracket \to  \VV_{>} \llbracket X^2 \rrbracket \quad\text{in}\quad C_{T}\CC^{\alpha-2},
 &\qquad  \VV_{\leqslant} \llbracket X_{\varepsilon}^2 \rrbracket\to  \VV_{\leqslant} \llbracket X^2 \rrbracket
  \quad\text{in}\quad C_{T}\CC^{\alpha} (\rho^{2}).
\end{align*}
Note that  we employed the same spaces as for the a priori estimates in Section \ref{ssec:apr}. As a consequence and in view of the estimates from Section \ref{ssec:apr}, we observe that there exists $K>0$ such that
$$\Phi_{\varepsilon}\to\Phi\quad\text{in}\quad C_{T}\CC^{\alpha-2}(\rho^{K}),\qquad \Psi_{\varepsilon}\to\Psi\quad\text{in}\quad C_{T}\CC^{\beta-\iota}(\rho^{K}),$$
where $\Phi,\Psi$ are defined as in \eqref{eq:44c}. The constant $K>0$ needs to be  chosen sufficiently large in order to compensate for the lack of convergence of $\phi_{\varepsilon}$ and $\psi_{\varepsilon}$ in $C_{T}L^{\infty}(\rho)$, which has to be replaced by $C_{T}L^{\infty}(\rho^{1+\delta})$ and $C_{T}L^{\infty}(\rho^{3+\beta+\delta})$, respectively. Passing to the limit in the remaining terms in \eqref{eq:42b} is straightforward, and therefore, the couple $(\phi,\psi)$ solves \eqref{eq:42b}, which is understood in distributional sense.

It remains to show that
$$(\phi,\psi)\in [C\CC^{\alpha}(\rho)\cap C^{\alpha/2}L^{\infty}(\rho)] \times[C\CC^{2+\beta}(\rho^{3+\beta})\cap C^{1}L^{\infty}(\rho^{3+\beta})\cap L^{\infty}L^{\infty}(\rho)].$$
To this end, we observe that according to the above convergence $(\phi_{\varepsilon},\psi_{\varepsilon})\to (\phi,\psi)$ it follows that $\Delta_{i}\phi_{\varepsilon}(t,x)\to \Delta_{i}\phi(t,x)$ and $\Delta_{i}\psi_{\varepsilon}(t,x)\to \Delta_{i}\psi(t,x)$ for every $i\geq -1$ and almost every $(t,x)\in [0,\8)\times\R^{2}$. In addition, the Littlewood--Paley blocks $\Delta_{i}\phi_{\varepsilon}, \Delta_{i}\psi_{\varepsilon}$ satisfy the uniform  bounds (even uniform in $\varepsilon,i$ and $t$)
$$
\|\rho_{t}\Delta_{i}\phi_{\varepsilon}(t)\|_{L^{\infty}}\lesssim 1,\qquad \|\rho_{t}\Delta_{i}\psi_{\varepsilon}(t)\|_{L^{\infty}}\lesssim 1.
$$
Consequently,  $\rho_{t}\Delta_{i}\phi_{\varepsilon}(t)\to \rho_{t}\Delta_{i}\phi(t)$ and $\rho_{t}\Delta_{i}\psi_{\varepsilon}(t)\to \rho_{t}\Delta_{i}\psi(t)$  weak-star in $L^{\infty}(\R^{2})$ for every $i\geq -1$ and almost every $t\in [0,\8)$. Since the $L^{\infty}$-norm is weak-star lower semicontinuous, we obtain
$$
\|\rho_{t}\Delta_{i}\phi(t)\|_{L^{\infty}}\leqslant \liminf_{\varepsilon\to 0}\|\rho_{t}\Delta_{i}\phi_{\varepsilon}(t)\|_{L^{\infty}}\leqslant \liminf_{\varepsilon\to 0}\|\phi_{\varepsilon}\|_{C_{T}\CC^{\alpha}(\rho)}2^{-i\alpha}\lesssim 2^{-i\alpha},
$$
$$
\|\rho_{t}\Delta_{i}\psi(t)\|_{L^{\infty}}\leqslant \liminf_{\varepsilon\to 0}\|\rho_{t}\Delta_{i}\psi_{\varepsilon}(t)\|_{L^{\infty}}\leqslant \liminf_{\varepsilon\to 0}\|\psi_{\varepsilon}\|_{C_{T}L^{\infty}(\rho)}\lesssim 1,
$$
and by the same argument
$$
\|\rho_{t}^{3+\beta}\Delta_{i}\psi(t)\|_{L^{\infty}}\leqslant \liminf_{\varepsilon\to 0}\|\rho_{t}^{3+\beta}\Delta_{i}\psi_{\varepsilon}(t)\|_{L^{\infty}}\leqslant \liminf_{\varepsilon\to 0}\|\psi_{\varepsilon}(t)\|_{C_{T}\CC^{2+\beta}(\rho^{3+\beta})}\lesssim 2^{-i(2+\beta)}.
$$
This implies that
$$(\phi,\psi)\in L^{\infty}\CC^{\alpha}(\rho) \times[L^{\infty}\CC^{2+\beta}(\rho^{3+\beta})\cap  L^{\infty}L^{\infty}(\rho)].$$
Now, using the convergence $(\phi_{\varepsilon},\psi_{\varepsilon})\to (\phi,\psi)$ in $C_{ \rm{loc}}^{\alpha/2}L^{\infty}(\rho)\times C^{1}_{ \rm{loc}}L^{\infty}(\rho)$ we obtain
$$
\|\phi(t)-\phi(s)\|_{L^{\infty}(\rho)}=\lim_{\varepsilon\to 0}\|\phi_{\varepsilon}(t)-\phi_{\varepsilon}(s)\|_{L^{\infty}(\rho)}\leqslant \lim_{\varepsilon\to 0}\|\phi_{\varepsilon}\|_{C^{\alpha/2}L^{\infty}(\rho)}|t-s|^{\alpha/2}\lesssim |t-s|^{\alpha/2}
$$
and similarly for the norm of $\psi$ in $C^{1}_{T}L^{\infty}(\rho^{3+\beta}).$ Hence
$$(\phi,\psi)\in C^{\alpha/2}_{T}L^{\infty}(\rho) \times C^{1}_{T}L^{\infty}(\rho^{{3+\beta}}).$$
Now,  we apply the Schauder estimates for both $\phi$ and $\psi$ (i.e. Lemma \ref{lemma:sch} and Lemma \ref{lemma:schauder-par}) to obtain continuity in time, namely,
$$
(\phi,\psi)\in C_{T}\CC^{\alpha}(\rho)\times C_{T}\CC^{2+\beta}(\rho^{3+\beta}).
$$
The proof is complete.
\end{proof}

\section{Parabolic $\Phi^4_3$ model}
\label{sec:43}

We proceed by similar arguments as  in the elliptic $\Phi^4_5$ model discussed in Section \ref{sec:45}. More precisely, we intend to study the parabolic equation
\begin{equation}\label{eq:43}
\partial_{t} \varphi+(- \Delta + \mu) \varphi + \varphi^3 + (- 3 a + 3b) \varphi - \xi = 0,\qquad\varphi(0)=\varphi_{0},
\end{equation}
in $\mathbbm{R}^3$ where $\xi$ is a space-time white noise, $a,b$ stand for renormalization constants. We recall that according to Theorem \ref{thm:renorm43} the renormalization constant $b$ depends on time and is bounded.
Without loss of generality we assume $\mu>0$ (see Remark~\ref{rem:mu}).

While the existence of solutions to the parabolic $\Phi^{4}$ model in dimension 2 was a more or less straightforward consequence of the elliptic a priori estimates for the $\Phi^{4}_{4}$ model from Section \ref{ssec:apr}, the situation is more involved in dimension $3$, which will be seen in the sequel. To be more precise, we let $(- \Delta +
\mu) = \Q$ and $\LL=\partial_{t}+\Q\hspace{-.2em}$
 and we introduce the ansatz
\[ \varphi = X - X^{\!\resizebox{0.6em}{!}{
\begin{tikzpicture}
\pgfpathmoveto{\pgfqpoint{0cm}{-0.035cm}}
\pgfpathlineto{\pgfqpoint{1.376cm}{-0.035cm}}
\pgfpathlineto{\pgfqpoint{1.376cm}{1.552cm}}
\pgfpathlineto{\pgfqpoint{0cm}{1.552cm}}
\pgfpathclose
\pgfusepath{clip}
\begin{pgfscope}
\begin{pgfscope}
\pgfpathmoveto{\pgfqpoint{0cm}{-0.035cm}}
\pgfpathlineto{\pgfqpoint{1.376cm}{-0.035cm}}
\pgfpathlineto{\pgfqpoint{1.376cm}{1.552cm}}
\pgfpathlineto{\pgfqpoint{0cm}{1.552cm}}
\pgfpathclose
\pgfusepath{clip}
\begin{pgfscope}
\begin{pgfscope}
\pgfsetdash{}{0cm}
\pgfsetlinewidth{0.818mm}
\pgfsetroundcap
\pgfsetroundjoin
\pgfsetmiterlimit{7.0}
\definecolor{eps2pgf_color}{gray}{0}\pgfsetstrokecolor{eps2pgf_color}\pgfsetfillcolor{eps2pgf_color}
\pgfpathmoveto{\pgfqpoint{0.117cm}{1.421cm}}
\pgfpathlineto{\pgfqpoint{0.682cm}{0.671cm}}
\pgfpathlineto{\pgfqpoint{1.246cm}{1.421cm}}
\pgfusepath{stroke}
\end{pgfscope}
\definecolor{eps2pgf_color}{gray}{0}\pgfsetstrokecolor{eps2pgf_color}\pgfsetfillcolor{eps2pgf_color}
\pgfpathmoveto{\pgfqpoint{0.273cm}{1.395cm}}
\pgfpathcurveto{\pgfqpoint{0.273cm}{1.432cm}}{\pgfqpoint{0.259cm}{1.467cm}}{\pgfqpoint{0.233cm}{1.492cm}}
\pgfpathcurveto{\pgfqpoint{0.207cm}{1.518cm}}{\pgfqpoint{0.173cm}{1.532cm}}{\pgfqpoint{0.137cm}{1.532cm}}
\pgfpathcurveto{\pgfqpoint{0.1cm}{1.532cm}}{\pgfqpoint{0.066cm}{1.518cm}}{\pgfqpoint{0.04cm}{1.492cm}}
\pgfpathcurveto{\pgfqpoint{0.014cm}{1.467cm}}{\pgfqpoint{0cm}{1.432cm}}{\pgfqpoint{0cm}{1.395cm}}
\pgfpathcurveto{\pgfqpoint{0cm}{1.359cm}}{\pgfqpoint{0.014cm}{1.324cm}}{\pgfqpoint{0.04cm}{1.299cm}}
\pgfpathcurveto{\pgfqpoint{0.066cm}{1.273cm}}{\pgfqpoint{0.1cm}{1.258cm}}{\pgfqpoint{0.137cm}{1.258cm}}
\pgfpathcurveto{\pgfqpoint{0.173cm}{1.258cm}}{\pgfqpoint{0.207cm}{1.273cm}}{\pgfqpoint{0.233cm}{1.299cm}}
\pgfpathcurveto{\pgfqpoint{0.259cm}{1.324cm}}{\pgfqpoint{0.273cm}{1.359cm}}{\pgfqpoint{0.273cm}{1.395cm}}
\pgfusepath{fill}
\begin{pgfscope}
\pgfsetdash{}{0cm}
\pgfsetlinewidth{0.818mm}
\pgfsetmiterlimit{7.0}
\pgfpathmoveto{\pgfqpoint{0.682cm}{0.671cm}}
\pgfpathlineto{\pgfqpoint{0.679cm}{1.418cm}}
\pgfusepath{stroke}
\end{pgfscope}
\pgfpathmoveto{\pgfqpoint{0.815cm}{1.399cm}}
\pgfpathcurveto{\pgfqpoint{0.815cm}{1.435cm}}{\pgfqpoint{0.801cm}{1.47cm}}{\pgfqpoint{0.775cm}{1.496cm}}
\pgfpathcurveto{\pgfqpoint{0.75cm}{1.521cm}}{\pgfqpoint{0.715cm}{1.536cm}}{\pgfqpoint{0.679cm}{1.536cm}}
\pgfpathcurveto{\pgfqpoint{0.643cm}{1.536cm}}{\pgfqpoint{0.608cm}{1.521cm}}{\pgfqpoint{0.582cm}{1.496cm}}
\pgfpathcurveto{\pgfqpoint{0.557cm}{1.47cm}}{\pgfqpoint{0.542cm}{1.435cm}}{\pgfqpoint{0.542cm}{1.399cm}}
\pgfpathcurveto{\pgfqpoint{0.542cm}{1.363cm}}{\pgfqpoint{0.557cm}{1.328cm}}{\pgfqpoint{0.582cm}{1.302cm}}
\pgfpathcurveto{\pgfqpoint{0.608cm}{1.276cm}}{\pgfqpoint{0.643cm}{1.262cm}}{\pgfqpoint{0.679cm}{1.262cm}}
\pgfpathcurveto{\pgfqpoint{0.715cm}{1.262cm}}{\pgfqpoint{0.75cm}{1.276cm}}{\pgfqpoint{0.775cm}{1.302cm}}
\pgfpathcurveto{\pgfqpoint{0.801cm}{1.328cm}}{\pgfqpoint{0.815cm}{1.363cm}}{\pgfqpoint{0.815cm}{1.399cm}}
\pgfusepath{fill}
\pgfpathmoveto{\pgfqpoint{1.345cm}{1.371cm}}
\pgfpathcurveto{\pgfqpoint{1.345cm}{1.408cm}}{\pgfqpoint{1.331cm}{1.442cm}}{\pgfqpoint{1.305cm}{1.468cm}}
\pgfpathcurveto{\pgfqpoint{1.28cm}{1.494cm}}{\pgfqpoint{1.245cm}{1.508cm}}{\pgfqpoint{1.209cm}{1.508cm}}
\pgfpathcurveto{\pgfqpoint{1.172cm}{1.508cm}}{\pgfqpoint{1.138cm}{1.494cm}}{\pgfqpoint{1.112cm}{1.468cm}}
\pgfpathcurveto{\pgfqpoint{1.087cm}{1.442cm}}{\pgfqpoint{1.072cm}{1.408cm}}{\pgfqpoint{1.072cm}{1.371cm}}
\pgfpathcurveto{\pgfqpoint{1.072cm}{1.335cm}}{\pgfqpoint{1.087cm}{1.3cm}}{\pgfqpoint{1.112cm}{1.274cm}}
\pgfpathcurveto{\pgfqpoint{1.138cm}{1.249cm}}{\pgfqpoint{1.172cm}{1.234cm}}{\pgfqpoint{1.209cm}{1.234cm}}
\pgfpathcurveto{\pgfqpoint{1.245cm}{1.234cm}}{\pgfqpoint{1.28cm}{1.249cm}}{\pgfqpoint{1.305cm}{1.274cm}}
\pgfpathcurveto{\pgfqpoint{1.331cm}{1.3cm}}{\pgfqpoint{1.345cm}{1.335cm}}{\pgfqpoint{1.345cm}{1.371cm}}
\pgfusepath{fill}
\begin{pgfscope}
\pgfsetdash{}{0cm}
\pgfsetlinewidth{0.818mm}
\pgfsetroundcap
\pgfsetmiterlimit{4.0}
\pgfpathmoveto{\pgfqpoint{0.682cm}{0.671cm}}
\pgfpathlineto{\pgfqpoint{0.682cm}{0.042cm}}
\pgfusepath{stroke}
\end{pgfscope}
\end{pgfscope}
\end{pgfscope}
\end{pgfscope}
\end{tikzpicture}}} + \phi + \psi\]
with $X$ being stationary and
\[\LL X = \xi, \quad \llbracket X^2
   \rrbracket \assign X^2 - a,\quad \llbracket X^3 \rrbracket \assign X^3 - 3 a X, \]
\[  \LL X^{\!\resizebox{0.6em}{!}{
\begin{tikzpicture}
\pgfpathmoveto{\pgfqpoint{0cm}{-0.035cm}}
\pgfpathlineto{\pgfqpoint{1.376cm}{-0.035cm}}
\pgfpathlineto{\pgfqpoint{1.376cm}{1.552cm}}
\pgfpathlineto{\pgfqpoint{0cm}{1.552cm}}
\pgfpathclose
\pgfusepath{clip}
\begin{pgfscope}
\begin{pgfscope}
\pgfpathmoveto{\pgfqpoint{0cm}{-0.035cm}}
\pgfpathlineto{\pgfqpoint{1.376cm}{-0.035cm}}
\pgfpathlineto{\pgfqpoint{1.376cm}{1.552cm}}
\pgfpathlineto{\pgfqpoint{0cm}{1.552cm}}
\pgfpathclose
\pgfusepath{clip}
\begin{pgfscope}
\begin{pgfscope}
\pgfsetdash{}{0cm}
\pgfsetlinewidth{0.818mm}
\pgfsetroundcap
\pgfsetroundjoin
\pgfsetmiterlimit{7.0}
\definecolor{eps2pgf_color}{gray}{0}\pgfsetstrokecolor{eps2pgf_color}\pgfsetfillcolor{eps2pgf_color}
\pgfpathmoveto{\pgfqpoint{0.117cm}{1.421cm}}
\pgfpathlineto{\pgfqpoint{0.682cm}{0.671cm}}
\pgfpathlineto{\pgfqpoint{1.246cm}{1.421cm}}
\pgfusepath{stroke}
\end{pgfscope}
\definecolor{eps2pgf_color}{gray}{0}\pgfsetstrokecolor{eps2pgf_color}\pgfsetfillcolor{eps2pgf_color}
\pgfpathmoveto{\pgfqpoint{0.273cm}{1.395cm}}
\pgfpathcurveto{\pgfqpoint{0.273cm}{1.432cm}}{\pgfqpoint{0.259cm}{1.467cm}}{\pgfqpoint{0.233cm}{1.492cm}}
\pgfpathcurveto{\pgfqpoint{0.207cm}{1.518cm}}{\pgfqpoint{0.173cm}{1.532cm}}{\pgfqpoint{0.137cm}{1.532cm}}
\pgfpathcurveto{\pgfqpoint{0.1cm}{1.532cm}}{\pgfqpoint{0.066cm}{1.518cm}}{\pgfqpoint{0.04cm}{1.492cm}}
\pgfpathcurveto{\pgfqpoint{0.014cm}{1.467cm}}{\pgfqpoint{0cm}{1.432cm}}{\pgfqpoint{0cm}{1.395cm}}
\pgfpathcurveto{\pgfqpoint{0cm}{1.359cm}}{\pgfqpoint{0.014cm}{1.324cm}}{\pgfqpoint{0.04cm}{1.299cm}}
\pgfpathcurveto{\pgfqpoint{0.066cm}{1.273cm}}{\pgfqpoint{0.1cm}{1.258cm}}{\pgfqpoint{0.137cm}{1.258cm}}
\pgfpathcurveto{\pgfqpoint{0.173cm}{1.258cm}}{\pgfqpoint{0.207cm}{1.273cm}}{\pgfqpoint{0.233cm}{1.299cm}}
\pgfpathcurveto{\pgfqpoint{0.259cm}{1.324cm}}{\pgfqpoint{0.273cm}{1.359cm}}{\pgfqpoint{0.273cm}{1.395cm}}
\pgfusepath{fill}
\begin{pgfscope}
\pgfsetdash{}{0cm}
\pgfsetlinewidth{0.818mm}
\pgfsetmiterlimit{7.0}
\pgfpathmoveto{\pgfqpoint{0.682cm}{0.671cm}}
\pgfpathlineto{\pgfqpoint{0.679cm}{1.418cm}}
\pgfusepath{stroke}
\end{pgfscope}
\pgfpathmoveto{\pgfqpoint{0.815cm}{1.399cm}}
\pgfpathcurveto{\pgfqpoint{0.815cm}{1.435cm}}{\pgfqpoint{0.801cm}{1.47cm}}{\pgfqpoint{0.775cm}{1.496cm}}
\pgfpathcurveto{\pgfqpoint{0.75cm}{1.521cm}}{\pgfqpoint{0.715cm}{1.536cm}}{\pgfqpoint{0.679cm}{1.536cm}}
\pgfpathcurveto{\pgfqpoint{0.643cm}{1.536cm}}{\pgfqpoint{0.608cm}{1.521cm}}{\pgfqpoint{0.582cm}{1.496cm}}
\pgfpathcurveto{\pgfqpoint{0.557cm}{1.47cm}}{\pgfqpoint{0.542cm}{1.435cm}}{\pgfqpoint{0.542cm}{1.399cm}}
\pgfpathcurveto{\pgfqpoint{0.542cm}{1.363cm}}{\pgfqpoint{0.557cm}{1.328cm}}{\pgfqpoint{0.582cm}{1.302cm}}
\pgfpathcurveto{\pgfqpoint{0.608cm}{1.276cm}}{\pgfqpoint{0.643cm}{1.262cm}}{\pgfqpoint{0.679cm}{1.262cm}}
\pgfpathcurveto{\pgfqpoint{0.715cm}{1.262cm}}{\pgfqpoint{0.75cm}{1.276cm}}{\pgfqpoint{0.775cm}{1.302cm}}
\pgfpathcurveto{\pgfqpoint{0.801cm}{1.328cm}}{\pgfqpoint{0.815cm}{1.363cm}}{\pgfqpoint{0.815cm}{1.399cm}}
\pgfusepath{fill}
\pgfpathmoveto{\pgfqpoint{1.345cm}{1.371cm}}
\pgfpathcurveto{\pgfqpoint{1.345cm}{1.408cm}}{\pgfqpoint{1.331cm}{1.442cm}}{\pgfqpoint{1.305cm}{1.468cm}}
\pgfpathcurveto{\pgfqpoint{1.28cm}{1.494cm}}{\pgfqpoint{1.245cm}{1.508cm}}{\pgfqpoint{1.209cm}{1.508cm}}
\pgfpathcurveto{\pgfqpoint{1.172cm}{1.508cm}}{\pgfqpoint{1.138cm}{1.494cm}}{\pgfqpoint{1.112cm}{1.468cm}}
\pgfpathcurveto{\pgfqpoint{1.087cm}{1.442cm}}{\pgfqpoint{1.072cm}{1.408cm}}{\pgfqpoint{1.072cm}{1.371cm}}
\pgfpathcurveto{\pgfqpoint{1.072cm}{1.335cm}}{\pgfqpoint{1.087cm}{1.3cm}}{\pgfqpoint{1.112cm}{1.274cm}}
\pgfpathcurveto{\pgfqpoint{1.138cm}{1.249cm}}{\pgfqpoint{1.172cm}{1.234cm}}{\pgfqpoint{1.209cm}{1.234cm}}
\pgfpathcurveto{\pgfqpoint{1.245cm}{1.234cm}}{\pgfqpoint{1.28cm}{1.249cm}}{\pgfqpoint{1.305cm}{1.274cm}}
\pgfpathcurveto{\pgfqpoint{1.331cm}{1.3cm}}{\pgfqpoint{1.345cm}{1.335cm}}{\pgfqpoint{1.345cm}{1.371cm}}
\pgfusepath{fill}
\begin{pgfscope}
\pgfsetdash{}{0cm}
\pgfsetlinewidth{0.818mm}
\pgfsetroundcap
\pgfsetmiterlimit{4.0}
\pgfpathmoveto{\pgfqpoint{0.682cm}{0.671cm}}
\pgfpathlineto{\pgfqpoint{0.682cm}{0.042cm}}
\pgfusepath{stroke}
\end{pgfscope}
\end{pgfscope}
\end{pgfscope}
\end{pgfscope}
\end{tikzpicture}}} = \llbracket X^3 \rrbracket, \quad X^{\!\resizebox{0.6em}{!}{
\begin{tikzpicture}
\pgfpathmoveto{\pgfqpoint{0cm}{-0.035cm}}
\pgfpathlineto{\pgfqpoint{1.376cm}{-0.035cm}}
\pgfpathlineto{\pgfqpoint{1.376cm}{1.552cm}}
\pgfpathlineto{\pgfqpoint{0cm}{1.552cm}}
\pgfpathclose
\pgfusepath{clip}
\begin{pgfscope}
\begin{pgfscope}
\pgfpathmoveto{\pgfqpoint{0cm}{-0.035cm}}
\pgfpathlineto{\pgfqpoint{1.376cm}{-0.035cm}}
\pgfpathlineto{\pgfqpoint{1.376cm}{1.552cm}}
\pgfpathlineto{\pgfqpoint{0cm}{1.552cm}}
\pgfpathclose
\pgfusepath{clip}
\begin{pgfscope}
\begin{pgfscope}
\pgfsetdash{}{0cm}
\pgfsetlinewidth{0.818mm}
\pgfsetroundcap
\pgfsetroundjoin
\pgfsetmiterlimit{7.0}
\definecolor{eps2pgf_color}{gray}{0}\pgfsetstrokecolor{eps2pgf_color}\pgfsetfillcolor{eps2pgf_color}
\pgfpathmoveto{\pgfqpoint{0.117cm}{1.421cm}}
\pgfpathlineto{\pgfqpoint{0.682cm}{0.671cm}}
\pgfpathlineto{\pgfqpoint{1.246cm}{1.421cm}}
\pgfusepath{stroke}
\end{pgfscope}
\definecolor{eps2pgf_color}{gray}{0}\pgfsetstrokecolor{eps2pgf_color}\pgfsetfillcolor{eps2pgf_color}
\pgfpathmoveto{\pgfqpoint{0.273cm}{1.395cm}}
\pgfpathcurveto{\pgfqpoint{0.273cm}{1.432cm}}{\pgfqpoint{0.259cm}{1.467cm}}{\pgfqpoint{0.233cm}{1.492cm}}
\pgfpathcurveto{\pgfqpoint{0.207cm}{1.518cm}}{\pgfqpoint{0.173cm}{1.532cm}}{\pgfqpoint{0.137cm}{1.532cm}}
\pgfpathcurveto{\pgfqpoint{0.1cm}{1.532cm}}{\pgfqpoint{0.066cm}{1.518cm}}{\pgfqpoint{0.04cm}{1.492cm}}
\pgfpathcurveto{\pgfqpoint{0.014cm}{1.467cm}}{\pgfqpoint{0cm}{1.432cm}}{\pgfqpoint{0cm}{1.395cm}}
\pgfpathcurveto{\pgfqpoint{0cm}{1.359cm}}{\pgfqpoint{0.014cm}{1.324cm}}{\pgfqpoint{0.04cm}{1.299cm}}
\pgfpathcurveto{\pgfqpoint{0.066cm}{1.273cm}}{\pgfqpoint{0.1cm}{1.258cm}}{\pgfqpoint{0.137cm}{1.258cm}}
\pgfpathcurveto{\pgfqpoint{0.173cm}{1.258cm}}{\pgfqpoint{0.207cm}{1.273cm}}{\pgfqpoint{0.233cm}{1.299cm}}
\pgfpathcurveto{\pgfqpoint{0.259cm}{1.324cm}}{\pgfqpoint{0.273cm}{1.359cm}}{\pgfqpoint{0.273cm}{1.395cm}}
\pgfusepath{fill}
\begin{pgfscope}
\pgfsetdash{}{0cm}
\pgfsetlinewidth{0.818mm}
\pgfsetmiterlimit{7.0}
\pgfpathmoveto{\pgfqpoint{0.682cm}{0.671cm}}
\pgfpathlineto{\pgfqpoint{0.679cm}{1.418cm}}
\pgfusepath{stroke}
\end{pgfscope}
\pgfpathmoveto{\pgfqpoint{0.815cm}{1.399cm}}
\pgfpathcurveto{\pgfqpoint{0.815cm}{1.435cm}}{\pgfqpoint{0.801cm}{1.47cm}}{\pgfqpoint{0.775cm}{1.496cm}}
\pgfpathcurveto{\pgfqpoint{0.75cm}{1.521cm}}{\pgfqpoint{0.715cm}{1.536cm}}{\pgfqpoint{0.679cm}{1.536cm}}
\pgfpathcurveto{\pgfqpoint{0.643cm}{1.536cm}}{\pgfqpoint{0.608cm}{1.521cm}}{\pgfqpoint{0.582cm}{1.496cm}}
\pgfpathcurveto{\pgfqpoint{0.557cm}{1.47cm}}{\pgfqpoint{0.542cm}{1.435cm}}{\pgfqpoint{0.542cm}{1.399cm}}
\pgfpathcurveto{\pgfqpoint{0.542cm}{1.363cm}}{\pgfqpoint{0.557cm}{1.328cm}}{\pgfqpoint{0.582cm}{1.302cm}}
\pgfpathcurveto{\pgfqpoint{0.608cm}{1.276cm}}{\pgfqpoint{0.643cm}{1.262cm}}{\pgfqpoint{0.679cm}{1.262cm}}
\pgfpathcurveto{\pgfqpoint{0.715cm}{1.262cm}}{\pgfqpoint{0.75cm}{1.276cm}}{\pgfqpoint{0.775cm}{1.302cm}}
\pgfpathcurveto{\pgfqpoint{0.801cm}{1.328cm}}{\pgfqpoint{0.815cm}{1.363cm}}{\pgfqpoint{0.815cm}{1.399cm}}
\pgfusepath{fill}
\pgfpathmoveto{\pgfqpoint{1.345cm}{1.371cm}}
\pgfpathcurveto{\pgfqpoint{1.345cm}{1.408cm}}{\pgfqpoint{1.331cm}{1.442cm}}{\pgfqpoint{1.305cm}{1.468cm}}
\pgfpathcurveto{\pgfqpoint{1.28cm}{1.494cm}}{\pgfqpoint{1.245cm}{1.508cm}}{\pgfqpoint{1.209cm}{1.508cm}}
\pgfpathcurveto{\pgfqpoint{1.172cm}{1.508cm}}{\pgfqpoint{1.138cm}{1.494cm}}{\pgfqpoint{1.112cm}{1.468cm}}
\pgfpathcurveto{\pgfqpoint{1.087cm}{1.442cm}}{\pgfqpoint{1.072cm}{1.408cm}}{\pgfqpoint{1.072cm}{1.371cm}}
\pgfpathcurveto{\pgfqpoint{1.072cm}{1.335cm}}{\pgfqpoint{1.087cm}{1.3cm}}{\pgfqpoint{1.112cm}{1.274cm}}
\pgfpathcurveto{\pgfqpoint{1.138cm}{1.249cm}}{\pgfqpoint{1.172cm}{1.234cm}}{\pgfqpoint{1.209cm}{1.234cm}}
\pgfpathcurveto{\pgfqpoint{1.245cm}{1.234cm}}{\pgfqpoint{1.28cm}{1.249cm}}{\pgfqpoint{1.305cm}{1.274cm}}
\pgfpathcurveto{\pgfqpoint{1.331cm}{1.3cm}}{\pgfqpoint{1.345cm}{1.335cm}}{\pgfqpoint{1.345cm}{1.371cm}}
\pgfusepath{fill}
\begin{pgfscope}
\pgfsetdash{}{0cm}
\pgfsetlinewidth{0.818mm}
\pgfsetroundcap
\pgfsetmiterlimit{4.0}
\pgfpathmoveto{\pgfqpoint{0.682cm}{0.671cm}}
\pgfpathlineto{\pgfqpoint{0.682cm}{0.042cm}}
\pgfusepath{stroke}
\end{pgfscope}
\end{pgfscope}
\end{pgfscope}
\end{pgfscope}
\end{tikzpicture}}}(0)=0,\qquad\LL
   X^{\!\resizebox{0.6em}{!}{
\begin{tikzpicture}
\pgfpathmoveto{\pgfqpoint{0cm}{0cm}}
\pgfpathlineto{\pgfqpoint{1.376cm}{0cm}}
\pgfpathlineto{\pgfqpoint{1.376cm}{1.588cm}}
\pgfpathlineto{\pgfqpoint{0cm}{1.588cm}}
\pgfpathclose
\pgfusepath{clip}
\begin{pgfscope}
\begin{pgfscope}
\pgfpathmoveto{\pgfqpoint{0cm}{0cm}}
\pgfpathlineto{\pgfqpoint{1.376cm}{0cm}}
\pgfpathlineto{\pgfqpoint{1.376cm}{1.588cm}}
\pgfpathlineto{\pgfqpoint{0cm}{1.588cm}}
\pgfpathclose
\pgfusepath{clip}
\begin{pgfscope}
\begin{pgfscope}
\definecolor{eps2pgf_color}{gray}{0.976471}\pgfsetstrokecolor{eps2pgf_color}\pgfsetfillcolor{eps2pgf_color}
\pgfpathmoveto{\pgfqpoint{0cm}{0cm}}
\pgfpathlineto{\pgfqpoint{1.376cm}{0cm}}
\pgfpathlineto{\pgfqpoint{1.376cm}{1.588cm}}
\pgfpathlineto{\pgfqpoint{0cm}{1.588cm}}
\pgfpathclose
\pgfusepath{fill}
\end{pgfscope}
\begin{pgfscope}
\pgfsetdash{}{0cm}
\pgfsetlinewidth{0.818mm}
\pgfsetroundcap
\pgfsetroundjoin
\pgfsetmiterlimit{7.0}
\definecolor{eps2pgf_color}{gray}{0}\pgfsetstrokecolor{eps2pgf_color}\pgfsetfillcolor{eps2pgf_color}
\pgfpathmoveto{\pgfqpoint{0.117cm}{1.476cm}}
\pgfpathlineto{\pgfqpoint{0.682cm}{0.726cm}}
\pgfpathlineto{\pgfqpoint{1.246cm}{1.476cm}}
\pgfusepath{stroke}
\end{pgfscope}
\definecolor{eps2pgf_color}{gray}{0}\pgfsetstrokecolor{eps2pgf_color}\pgfsetfillcolor{eps2pgf_color}
\pgfpathmoveto{\pgfqpoint{0.273cm}{1.451cm}}
\pgfpathcurveto{\pgfqpoint{0.273cm}{1.487cm}}{\pgfqpoint{0.259cm}{1.522cm}}{\pgfqpoint{0.233cm}{1.547cm}}
\pgfpathcurveto{\pgfqpoint{0.207cm}{1.573cm}}{\pgfqpoint{0.173cm}{1.588cm}}{\pgfqpoint{0.137cm}{1.588cm}}
\pgfpathcurveto{\pgfqpoint{0.1cm}{1.588cm}}{\pgfqpoint{0.066cm}{1.573cm}}{\pgfqpoint{0.04cm}{1.547cm}}
\pgfpathcurveto{\pgfqpoint{0.014cm}{1.522cm}}{\pgfqpoint{0cm}{1.487cm}}{\pgfqpoint{0cm}{1.451cm}}
\pgfpathcurveto{\pgfqpoint{0cm}{1.414cm}}{\pgfqpoint{0.014cm}{1.379cm}}{\pgfqpoint{0.04cm}{1.354cm}}
\pgfpathcurveto{\pgfqpoint{0.066cm}{1.328cm}}{\pgfqpoint{0.1cm}{1.314cm}}{\pgfqpoint{0.137cm}{1.314cm}}
\pgfpathcurveto{\pgfqpoint{0.173cm}{1.314cm}}{\pgfqpoint{0.207cm}{1.328cm}}{\pgfqpoint{0.233cm}{1.354cm}}
\pgfpathcurveto{\pgfqpoint{0.259cm}{1.379cm}}{\pgfqpoint{0.273cm}{1.414cm}}{\pgfqpoint{0.273cm}{1.451cm}}
\pgfusepath{fill}
\pgfpathmoveto{\pgfqpoint{1.345cm}{1.426cm}}
\pgfpathcurveto{\pgfqpoint{1.345cm}{1.463cm}}{\pgfqpoint{1.331cm}{1.497cm}}{\pgfqpoint{1.305cm}{1.523cm}}
\pgfpathcurveto{\pgfqpoint{1.28cm}{1.549cm}}{\pgfqpoint{1.245cm}{1.563cm}}{\pgfqpoint{1.209cm}{1.563cm}}
\pgfpathcurveto{\pgfqpoint{1.172cm}{1.563cm}}{\pgfqpoint{1.138cm}{1.549cm}}{\pgfqpoint{1.112cm}{1.523cm}}
\pgfpathcurveto{\pgfqpoint{1.087cm}{1.497cm}}{\pgfqpoint{1.072cm}{1.463cm}}{\pgfqpoint{1.072cm}{1.426cm}}
\pgfpathcurveto{\pgfqpoint{1.072cm}{1.39cm}}{\pgfqpoint{1.087cm}{1.355cm}}{\pgfqpoint{1.112cm}{1.329cm}}
\pgfpathcurveto{\pgfqpoint{1.138cm}{1.304cm}}{\pgfqpoint{1.172cm}{1.289cm}}{\pgfqpoint{1.209cm}{1.289cm}}
\pgfpathcurveto{\pgfqpoint{1.245cm}{1.289cm}}{\pgfqpoint{1.28cm}{1.304cm}}{\pgfqpoint{1.305cm}{1.329cm}}
\pgfpathcurveto{\pgfqpoint{1.331cm}{1.355cm}}{\pgfqpoint{1.345cm}{1.39cm}}{\pgfqpoint{1.345cm}{1.426cm}}
\pgfusepath{fill}
\begin{pgfscope}
\pgfsetdash{}{0cm}
\pgfsetlinewidth{0.818mm}
\pgfsetroundcap
\pgfsetmiterlimit{4.0}
\pgfpathmoveto{\pgfqpoint{0.682cm}{0.726cm}}
\pgfpathlineto{\pgfqpoint{0.682cm}{0.097cm}}
\pgfusepath{stroke}
\end{pgfscope}
\end{pgfscope}
\end{pgfscope}
\end{pgfscope}
\end{tikzpicture}}} = \llbracket X^2 \rrbracket,\quadX^{\!\resizebox{0.6em}{!}{
\begin{tikzpicture}
\pgfpathmoveto{\pgfqpoint{0cm}{0cm}}
\pgfpathlineto{\pgfqpoint{1.376cm}{0cm}}
\pgfpathlineto{\pgfqpoint{1.376cm}{1.588cm}}
\pgfpathlineto{\pgfqpoint{0cm}{1.588cm}}
\pgfpathclose
\pgfusepath{clip}
\begin{pgfscope}
\begin{pgfscope}
\pgfpathmoveto{\pgfqpoint{0cm}{0cm}}
\pgfpathlineto{\pgfqpoint{1.376cm}{0cm}}
\pgfpathlineto{\pgfqpoint{1.376cm}{1.588cm}}
\pgfpathlineto{\pgfqpoint{0cm}{1.588cm}}
\pgfpathclose
\pgfusepath{clip}
\begin{pgfscope}
\begin{pgfscope}
\definecolor{eps2pgf_color}{gray}{0.976471}\pgfsetstrokecolor{eps2pgf_color}\pgfsetfillcolor{eps2pgf_color}
\pgfpathmoveto{\pgfqpoint{0cm}{0cm}}
\pgfpathlineto{\pgfqpoint{1.376cm}{0cm}}
\pgfpathlineto{\pgfqpoint{1.376cm}{1.588cm}}
\pgfpathlineto{\pgfqpoint{0cm}{1.588cm}}
\pgfpathclose
\pgfusepath{fill}
\end{pgfscope}
\begin{pgfscope}
\pgfsetdash{}{0cm}
\pgfsetlinewidth{0.818mm}
\pgfsetroundcap
\pgfsetroundjoin
\pgfsetmiterlimit{7.0}
\definecolor{eps2pgf_color}{gray}{0}\pgfsetstrokecolor{eps2pgf_color}\pgfsetfillcolor{eps2pgf_color}
\pgfpathmoveto{\pgfqpoint{0.117cm}{1.476cm}}
\pgfpathlineto{\pgfqpoint{0.682cm}{0.726cm}}
\pgfpathlineto{\pgfqpoint{1.246cm}{1.476cm}}
\pgfusepath{stroke}
\end{pgfscope}
\definecolor{eps2pgf_color}{gray}{0}\pgfsetstrokecolor{eps2pgf_color}\pgfsetfillcolor{eps2pgf_color}
\pgfpathmoveto{\pgfqpoint{0.273cm}{1.451cm}}
\pgfpathcurveto{\pgfqpoint{0.273cm}{1.487cm}}{\pgfqpoint{0.259cm}{1.522cm}}{\pgfqpoint{0.233cm}{1.547cm}}
\pgfpathcurveto{\pgfqpoint{0.207cm}{1.573cm}}{\pgfqpoint{0.173cm}{1.588cm}}{\pgfqpoint{0.137cm}{1.588cm}}
\pgfpathcurveto{\pgfqpoint{0.1cm}{1.588cm}}{\pgfqpoint{0.066cm}{1.573cm}}{\pgfqpoint{0.04cm}{1.547cm}}
\pgfpathcurveto{\pgfqpoint{0.014cm}{1.522cm}}{\pgfqpoint{0cm}{1.487cm}}{\pgfqpoint{0cm}{1.451cm}}
\pgfpathcurveto{\pgfqpoint{0cm}{1.414cm}}{\pgfqpoint{0.014cm}{1.379cm}}{\pgfqpoint{0.04cm}{1.354cm}}
\pgfpathcurveto{\pgfqpoint{0.066cm}{1.328cm}}{\pgfqpoint{0.1cm}{1.314cm}}{\pgfqpoint{0.137cm}{1.314cm}}
\pgfpathcurveto{\pgfqpoint{0.173cm}{1.314cm}}{\pgfqpoint{0.207cm}{1.328cm}}{\pgfqpoint{0.233cm}{1.354cm}}
\pgfpathcurveto{\pgfqpoint{0.259cm}{1.379cm}}{\pgfqpoint{0.273cm}{1.414cm}}{\pgfqpoint{0.273cm}{1.451cm}}
\pgfusepath{fill}
\pgfpathmoveto{\pgfqpoint{1.345cm}{1.426cm}}
\pgfpathcurveto{\pgfqpoint{1.345cm}{1.463cm}}{\pgfqpoint{1.331cm}{1.497cm}}{\pgfqpoint{1.305cm}{1.523cm}}
\pgfpathcurveto{\pgfqpoint{1.28cm}{1.549cm}}{\pgfqpoint{1.245cm}{1.563cm}}{\pgfqpoint{1.209cm}{1.563cm}}
\pgfpathcurveto{\pgfqpoint{1.172cm}{1.563cm}}{\pgfqpoint{1.138cm}{1.549cm}}{\pgfqpoint{1.112cm}{1.523cm}}
\pgfpathcurveto{\pgfqpoint{1.087cm}{1.497cm}}{\pgfqpoint{1.072cm}{1.463cm}}{\pgfqpoint{1.072cm}{1.426cm}}
\pgfpathcurveto{\pgfqpoint{1.072cm}{1.39cm}}{\pgfqpoint{1.087cm}{1.355cm}}{\pgfqpoint{1.112cm}{1.329cm}}
\pgfpathcurveto{\pgfqpoint{1.138cm}{1.304cm}}{\pgfqpoint{1.172cm}{1.289cm}}{\pgfqpoint{1.209cm}{1.289cm}}
\pgfpathcurveto{\pgfqpoint{1.245cm}{1.289cm}}{\pgfqpoint{1.28cm}{1.304cm}}{\pgfqpoint{1.305cm}{1.329cm}}
\pgfpathcurveto{\pgfqpoint{1.331cm}{1.355cm}}{\pgfqpoint{1.345cm}{1.39cm}}{\pgfqpoint{1.345cm}{1.426cm}}
\pgfusepath{fill}
\begin{pgfscope}
\pgfsetdash{}{0cm}
\pgfsetlinewidth{0.818mm}
\pgfsetroundcap
\pgfsetmiterlimit{4.0}
\pgfpathmoveto{\pgfqpoint{0.682cm}{0.726cm}}
\pgfpathlineto{\pgfqpoint{0.682cm}{0.097cm}}
\pgfusepath{stroke}
\end{pgfscope}
\end{pgfscope}
\end{pgfscope}
\end{pgfscope}
\end{tikzpicture}}}(0)=0, \]
\[X^{\!\resizebox{!}{.8em}{
\begin{tikzpicture}
\pgfpathmoveto{\pgfqpoint{0cm}{-0.035cm}}
\pgfpathlineto{\pgfqpoint{1.976cm}{-0.035cm}}
\pgfpathlineto{\pgfqpoint{1.976cm}{1.94cm}}
\pgfpathlineto{\pgfqpoint{0cm}{1.94cm}}
\pgfpathclose
\pgfusepath{clip}
\begin{pgfscope}
\begin{pgfscope}
\pgfpathmoveto{\pgfqpoint{0cm}{-0.035cm}}
\pgfpathlineto{\pgfqpoint{1.976cm}{-0.035cm}}
\pgfpathlineto{\pgfqpoint{1.976cm}{1.94cm}}
\pgfpathlineto{\pgfqpoint{0cm}{1.94cm}}
\pgfpathclose
\pgfusepath{clip}
\begin{pgfscope}
\begin{pgfscope}
\pgfsetdash{}{0cm}
\pgfsetlinewidth{0.818mm}
\pgfsetroundcap
\pgfsetroundjoin
\pgfsetmiterlimit{7.0}
\definecolor{eps2pgf_color}{gray}{0}\pgfsetstrokecolor{eps2pgf_color}\pgfsetfillcolor{eps2pgf_color}
\pgfpathmoveto{\pgfqpoint{0.117cm}{1.815cm}}
\pgfpathlineto{\pgfqpoint{0.682cm}{1.065cm}}
\pgfpathlineto{\pgfqpoint{1.246cm}{1.815cm}}
\pgfusepath{stroke}
\end{pgfscope}
\definecolor{eps2pgf_color}{gray}{0}\pgfsetstrokecolor{eps2pgf_color}\pgfsetfillcolor{eps2pgf_color}
\pgfpathmoveto{\pgfqpoint{0.273cm}{1.789cm}}
\pgfpathcurveto{\pgfqpoint{0.273cm}{1.825cm}}{\pgfqpoint{0.259cm}{1.86cm}}{\pgfqpoint{0.233cm}{1.886cm}}
\pgfpathcurveto{\pgfqpoint{0.207cm}{1.912cm}}{\pgfqpoint{0.173cm}{1.926cm}}{\pgfqpoint{0.137cm}{1.926cm}}
\pgfpathcurveto{\pgfqpoint{0.1cm}{1.926cm}}{\pgfqpoint{0.066cm}{1.912cm}}{\pgfqpoint{0.04cm}{1.886cm}}
\pgfpathcurveto{\pgfqpoint{0.014cm}{1.86cm}}{\pgfqpoint{0cm}{1.825cm}}{\pgfqpoint{0cm}{1.789cm}}
\pgfpathcurveto{\pgfqpoint{0cm}{1.753cm}}{\pgfqpoint{0.014cm}{1.718cm}}{\pgfqpoint{0.04cm}{1.692cm}}
\pgfpathcurveto{\pgfqpoint{0.066cm}{1.667cm}}{\pgfqpoint{0.1cm}{1.652cm}}{\pgfqpoint{0.137cm}{1.652cm}}
\pgfpathcurveto{\pgfqpoint{0.173cm}{1.652cm}}{\pgfqpoint{0.207cm}{1.667cm}}{\pgfqpoint{0.233cm}{1.692cm}}
\pgfpathcurveto{\pgfqpoint{0.259cm}{1.718cm}}{\pgfqpoint{0.273cm}{1.753cm}}{\pgfqpoint{0.273cm}{1.789cm}}
\pgfusepath{fill}
\begin{pgfscope}
\pgfsetdash{}{0cm}
\pgfsetlinewidth{0.818mm}
\pgfsetmiterlimit{7.0}
\pgfpathmoveto{\pgfqpoint{0.682cm}{1.065cm}}
\pgfpathlineto{\pgfqpoint{0.679cm}{1.812cm}}
\pgfusepath{stroke}
\end{pgfscope}
\pgfpathmoveto{\pgfqpoint{0.815cm}{1.793cm}}
\pgfpathcurveto{\pgfqpoint{0.815cm}{1.829cm}}{\pgfqpoint{0.801cm}{1.864cm}}{\pgfqpoint{0.775cm}{1.89cm}}
\pgfpathcurveto{\pgfqpoint{0.75cm}{1.915cm}}{\pgfqpoint{0.715cm}{1.93cm}}{\pgfqpoint{0.679cm}{1.93cm}}
\pgfpathcurveto{\pgfqpoint{0.643cm}{1.93cm}}{\pgfqpoint{0.608cm}{1.915cm}}{\pgfqpoint{0.582cm}{1.89cm}}
\pgfpathcurveto{\pgfqpoint{0.557cm}{1.864cm}}{\pgfqpoint{0.542cm}{1.829cm}}{\pgfqpoint{0.542cm}{1.793cm}}
\pgfpathcurveto{\pgfqpoint{0.542cm}{1.756cm}}{\pgfqpoint{0.557cm}{1.722cm}}{\pgfqpoint{0.582cm}{1.696cm}}
\pgfpathcurveto{\pgfqpoint{0.608cm}{1.67cm}}{\pgfqpoint{0.643cm}{1.656cm}}{\pgfqpoint{0.679cm}{1.656cm}}
\pgfpathcurveto{\pgfqpoint{0.715cm}{1.656cm}}{\pgfqpoint{0.75cm}{1.67cm}}{\pgfqpoint{0.775cm}{1.696cm}}
\pgfpathcurveto{\pgfqpoint{0.801cm}{1.722cm}}{\pgfqpoint{0.815cm}{1.756cm}}{\pgfqpoint{0.815cm}{1.793cm}}
\pgfusepath{fill}
\pgfpathmoveto{\pgfqpoint{1.345cm}{1.765cm}}
\pgfpathcurveto{\pgfqpoint{1.345cm}{1.801cm}}{\pgfqpoint{1.331cm}{1.836cm}}{\pgfqpoint{1.305cm}{1.862cm}}
\pgfpathcurveto{\pgfqpoint{1.28cm}{1.887cm}}{\pgfqpoint{1.245cm}{1.902cm}}{\pgfqpoint{1.209cm}{1.902cm}}
\pgfpathcurveto{\pgfqpoint{1.172cm}{1.902cm}}{\pgfqpoint{1.138cm}{1.887cm}}{\pgfqpoint{1.112cm}{1.862cm}}
\pgfpathcurveto{\pgfqpoint{1.087cm}{1.836cm}}{\pgfqpoint{1.072cm}{1.801cm}}{\pgfqpoint{1.072cm}{1.765cm}}
\pgfpathcurveto{\pgfqpoint{1.072cm}{1.728cm}}{\pgfqpoint{1.087cm}{1.694cm}}{\pgfqpoint{1.112cm}{1.668cm}}
\pgfpathcurveto{\pgfqpoint{1.138cm}{1.642cm}}{\pgfqpoint{1.172cm}{1.628cm}}{\pgfqpoint{1.209cm}{1.628cm}}
\pgfpathcurveto{\pgfqpoint{1.245cm}{1.628cm}}{\pgfqpoint{1.28cm}{1.642cm}}{\pgfqpoint{1.305cm}{1.668cm}}
\pgfpathcurveto{\pgfqpoint{1.331cm}{1.694cm}}{\pgfqpoint{1.345cm}{1.728cm}}{\pgfqpoint{1.345cm}{1.765cm}}
\pgfusepath{fill}
\begin{pgfscope}
\pgfsetdash{}{0cm}
\pgfsetlinewidth{0.818mm}
\pgfsetroundcap
\pgfsetroundjoin
\pgfsetmiterlimit{7.0}
\pgfpathmoveto{\pgfqpoint{0.682cm}{1.065cm}}
\pgfpathlineto{\pgfqpoint{1.246cm}{0.315cm}}
\pgfpathlineto{\pgfqpoint{1.811cm}{1.065cm}}
\pgfusepath{stroke}
\end{pgfscope}
\pgfpathmoveto{\pgfqpoint{1.948cm}{1.065cm}}
\pgfpathcurveto{\pgfqpoint{1.948cm}{1.101cm}}{\pgfqpoint{1.933cm}{1.136cm}}{\pgfqpoint{1.907cm}{1.162cm}}
\pgfpathcurveto{\pgfqpoint{1.882cm}{1.187cm}}{\pgfqpoint{1.847cm}{1.202cm}}{\pgfqpoint{1.811cm}{1.202cm}}
\pgfpathcurveto{\pgfqpoint{1.775cm}{1.202cm}}{\pgfqpoint{1.74cm}{1.187cm}}{\pgfqpoint{1.714cm}{1.162cm}}
\pgfpathcurveto{\pgfqpoint{1.689cm}{1.136cm}}{\pgfqpoint{1.674cm}{1.101cm}}{\pgfqpoint{1.674cm}{1.065cm}}
\pgfpathcurveto{\pgfqpoint{1.674cm}{1.029cm}}{\pgfqpoint{1.689cm}{0.994cm}}{\pgfqpoint{1.714cm}{0.968cm}}
\pgfpathcurveto{\pgfqpoint{1.74cm}{0.942cm}}{\pgfqpoint{1.775cm}{0.928cm}}{\pgfqpoint{1.811cm}{0.928cm}}
\pgfpathcurveto{\pgfqpoint{1.847cm}{0.928cm}}{\pgfqpoint{1.882cm}{0.942cm}}{\pgfqpoint{1.907cm}{0.968cm}}
\pgfpathcurveto{\pgfqpoint{1.933cm}{0.994cm}}{\pgfqpoint{1.948cm}{1.029cm}}{\pgfqpoint{1.948cm}{1.065cm}}
\pgfusepath{fill}
\begin{pgfscope}
\pgfsetdash{}{0cm}
\pgfsetlinewidth{0.818mm}
\pgfsetmiterlimit{4.0}
\pgfpathmoveto{\pgfqpoint{1.383cm}{0.178cm}}
\pgfpathcurveto{\pgfqpoint{1.383cm}{0.214cm}}{\pgfqpoint{1.369cm}{0.249cm}}{\pgfqpoint{1.343cm}{0.275cm}}
\pgfpathcurveto{\pgfqpoint{1.317cm}{0.3cm}}{\pgfqpoint{1.283cm}{0.315cm}}{\pgfqpoint{1.246cm}{0.315cm}}
\pgfpathcurveto{\pgfqpoint{1.21cm}{0.315cm}}{\pgfqpoint{1.175cm}{0.3cm}}{\pgfqpoint{1.15cm}{0.275cm}}
\pgfpathcurveto{\pgfqpoint{1.124cm}{0.249cm}}{\pgfqpoint{1.11cm}{0.214cm}}{\pgfqpoint{1.11cm}{0.178cm}}
\pgfpathcurveto{\pgfqpoint{1.11cm}{0.141cm}}{\pgfqpoint{1.124cm}{0.107cm}}{\pgfqpoint{1.15cm}{0.081cm}}
\pgfpathcurveto{\pgfqpoint{1.175cm}{0.055cm}}{\pgfqpoint{1.21cm}{0.041cm}}{\pgfqpoint{1.246cm}{0.041cm}}
\pgfpathcurveto{\pgfqpoint{1.283cm}{0.041cm}}{\pgfqpoint{1.317cm}{0.055cm}}{\pgfqpoint{1.343cm}{0.081cm}}
\pgfpathcurveto{\pgfqpoint{1.369cm}{0.107cm}}{\pgfqpoint{1.383cm}{0.141cm}}{\pgfqpoint{1.383cm}{0.178cm}}
\pgfusepath{stroke}
\end{pgfscope}
\end{pgfscope}
\end{pgfscope}
\end{pgfscope}
\end{tikzpicture}}}=X^{\!\resizebox{0.6em}{!}{
\begin{tikzpicture}
\pgfpathmoveto{\pgfqpoint{0cm}{-0.035cm}}
\pgfpathlineto{\pgfqpoint{1.376cm}{-0.035cm}}
\pgfpathlineto{\pgfqpoint{1.376cm}{1.552cm}}
\pgfpathlineto{\pgfqpoint{0cm}{1.552cm}}
\pgfpathclose
\pgfusepath{clip}
\begin{pgfscope}
\begin{pgfscope}
\pgfpathmoveto{\pgfqpoint{0cm}{-0.035cm}}
\pgfpathlineto{\pgfqpoint{1.376cm}{-0.035cm}}
\pgfpathlineto{\pgfqpoint{1.376cm}{1.552cm}}
\pgfpathlineto{\pgfqpoint{0cm}{1.552cm}}
\pgfpathclose
\pgfusepath{clip}
\begin{pgfscope}
\begin{pgfscope}
\pgfsetdash{}{0cm}
\pgfsetlinewidth{0.818mm}
\pgfsetroundcap
\pgfsetroundjoin
\pgfsetmiterlimit{7.0}
\definecolor{eps2pgf_color}{gray}{0}\pgfsetstrokecolor{eps2pgf_color}\pgfsetfillcolor{eps2pgf_color}
\pgfpathmoveto{\pgfqpoint{0.117cm}{1.421cm}}
\pgfpathlineto{\pgfqpoint{0.682cm}{0.671cm}}
\pgfpathlineto{\pgfqpoint{1.246cm}{1.421cm}}
\pgfusepath{stroke}
\end{pgfscope}
\definecolor{eps2pgf_color}{gray}{0}\pgfsetstrokecolor{eps2pgf_color}\pgfsetfillcolor{eps2pgf_color}
\pgfpathmoveto{\pgfqpoint{0.273cm}{1.395cm}}
\pgfpathcurveto{\pgfqpoint{0.273cm}{1.432cm}}{\pgfqpoint{0.259cm}{1.467cm}}{\pgfqpoint{0.233cm}{1.492cm}}
\pgfpathcurveto{\pgfqpoint{0.207cm}{1.518cm}}{\pgfqpoint{0.173cm}{1.532cm}}{\pgfqpoint{0.137cm}{1.532cm}}
\pgfpathcurveto{\pgfqpoint{0.1cm}{1.532cm}}{\pgfqpoint{0.066cm}{1.518cm}}{\pgfqpoint{0.04cm}{1.492cm}}
\pgfpathcurveto{\pgfqpoint{0.014cm}{1.467cm}}{\pgfqpoint{0cm}{1.432cm}}{\pgfqpoint{0cm}{1.395cm}}
\pgfpathcurveto{\pgfqpoint{0cm}{1.359cm}}{\pgfqpoint{0.014cm}{1.324cm}}{\pgfqpoint{0.04cm}{1.299cm}}
\pgfpathcurveto{\pgfqpoint{0.066cm}{1.273cm}}{\pgfqpoint{0.1cm}{1.258cm}}{\pgfqpoint{0.137cm}{1.258cm}}
\pgfpathcurveto{\pgfqpoint{0.173cm}{1.258cm}}{\pgfqpoint{0.207cm}{1.273cm}}{\pgfqpoint{0.233cm}{1.299cm}}
\pgfpathcurveto{\pgfqpoint{0.259cm}{1.324cm}}{\pgfqpoint{0.273cm}{1.359cm}}{\pgfqpoint{0.273cm}{1.395cm}}
\pgfusepath{fill}
\begin{pgfscope}
\pgfsetdash{}{0cm}
\pgfsetlinewidth{0.818mm}
\pgfsetmiterlimit{7.0}
\pgfpathmoveto{\pgfqpoint{0.682cm}{0.671cm}}
\pgfpathlineto{\pgfqpoint{0.679cm}{1.418cm}}
\pgfusepath{stroke}
\end{pgfscope}
\pgfpathmoveto{\pgfqpoint{0.815cm}{1.399cm}}
\pgfpathcurveto{\pgfqpoint{0.815cm}{1.435cm}}{\pgfqpoint{0.801cm}{1.47cm}}{\pgfqpoint{0.775cm}{1.496cm}}
\pgfpathcurveto{\pgfqpoint{0.75cm}{1.521cm}}{\pgfqpoint{0.715cm}{1.536cm}}{\pgfqpoint{0.679cm}{1.536cm}}
\pgfpathcurveto{\pgfqpoint{0.643cm}{1.536cm}}{\pgfqpoint{0.608cm}{1.521cm}}{\pgfqpoint{0.582cm}{1.496cm}}
\pgfpathcurveto{\pgfqpoint{0.557cm}{1.47cm}}{\pgfqpoint{0.542cm}{1.435cm}}{\pgfqpoint{0.542cm}{1.399cm}}
\pgfpathcurveto{\pgfqpoint{0.542cm}{1.363cm}}{\pgfqpoint{0.557cm}{1.328cm}}{\pgfqpoint{0.582cm}{1.302cm}}
\pgfpathcurveto{\pgfqpoint{0.608cm}{1.276cm}}{\pgfqpoint{0.643cm}{1.262cm}}{\pgfqpoint{0.679cm}{1.262cm}}
\pgfpathcurveto{\pgfqpoint{0.715cm}{1.262cm}}{\pgfqpoint{0.75cm}{1.276cm}}{\pgfqpoint{0.775cm}{1.302cm}}
\pgfpathcurveto{\pgfqpoint{0.801cm}{1.328cm}}{\pgfqpoint{0.815cm}{1.363cm}}{\pgfqpoint{0.815cm}{1.399cm}}
\pgfusepath{fill}
\pgfpathmoveto{\pgfqpoint{1.345cm}{1.371cm}}
\pgfpathcurveto{\pgfqpoint{1.345cm}{1.408cm}}{\pgfqpoint{1.331cm}{1.442cm}}{\pgfqpoint{1.305cm}{1.468cm}}
\pgfpathcurveto{\pgfqpoint{1.28cm}{1.494cm}}{\pgfqpoint{1.245cm}{1.508cm}}{\pgfqpoint{1.209cm}{1.508cm}}
\pgfpathcurveto{\pgfqpoint{1.172cm}{1.508cm}}{\pgfqpoint{1.138cm}{1.494cm}}{\pgfqpoint{1.112cm}{1.468cm}}
\pgfpathcurveto{\pgfqpoint{1.087cm}{1.442cm}}{\pgfqpoint{1.072cm}{1.408cm}}{\pgfqpoint{1.072cm}{1.371cm}}
\pgfpathcurveto{\pgfqpoint{1.072cm}{1.335cm}}{\pgfqpoint{1.087cm}{1.3cm}}{\pgfqpoint{1.112cm}{1.274cm}}
\pgfpathcurveto{\pgfqpoint{1.138cm}{1.249cm}}{\pgfqpoint{1.172cm}{1.234cm}}{\pgfqpoint{1.209cm}{1.234cm}}
\pgfpathcurveto{\pgfqpoint{1.245cm}{1.234cm}}{\pgfqpoint{1.28cm}{1.249cm}}{\pgfqpoint{1.305cm}{1.274cm}}
\pgfpathcurveto{\pgfqpoint{1.331cm}{1.3cm}}{\pgfqpoint{1.345cm}{1.335cm}}{\pgfqpoint{1.345cm}{1.371cm}}
\pgfusepath{fill}
\begin{pgfscope}
\pgfsetdash{}{0cm}
\pgfsetlinewidth{0.818mm}
\pgfsetroundcap
\pgfsetmiterlimit{4.0}
\pgfpathmoveto{\pgfqpoint{0.682cm}{0.671cm}}
\pgfpathlineto{\pgfqpoint{0.682cm}{0.042cm}}
\pgfusepath{stroke}
\end{pgfscope}
\end{pgfscope}
\end{pgfscope}
\end{pgfscope}
\end{tikzpicture}}}\circ X,\qquad X^{\!\resizebox{!}{.8em}{
\begin{tikzpicture}
\pgfpathmoveto{\pgfqpoint{0cm}{-0.035cm}}
\pgfpathlineto{\pgfqpoint{1.976cm}{-0.035cm}}
\pgfpathlineto{\pgfqpoint{1.976cm}{1.94cm}}
\pgfpathlineto{\pgfqpoint{0cm}{1.94cm}}
\pgfpathclose
\pgfusepath{clip}
\begin{pgfscope}
\begin{pgfscope}
\pgfpathmoveto{\pgfqpoint{0cm}{-0.035cm}}
\pgfpathlineto{\pgfqpoint{1.976cm}{-0.035cm}}
\pgfpathlineto{\pgfqpoint{1.976cm}{1.94cm}}
\pgfpathlineto{\pgfqpoint{0cm}{1.94cm}}
\pgfpathclose
\pgfusepath{clip}
\begin{pgfscope}
\begin{pgfscope}
\pgfsetdash{}{0cm}
\pgfsetlinewidth{0.818mm}
\pgfsetroundcap
\pgfsetroundjoin
\pgfsetmiterlimit{7.0}
\definecolor{eps2pgf_color}{gray}{0}\pgfsetstrokecolor{eps2pgf_color}\pgfsetfillcolor{eps2pgf_color}
\pgfpathmoveto{\pgfqpoint{0.117cm}{1.815cm}}
\pgfpathlineto{\pgfqpoint{0.682cm}{1.065cm}}
\pgfpathlineto{\pgfqpoint{1.246cm}{1.815cm}}
\pgfusepath{stroke}
\end{pgfscope}
\definecolor{eps2pgf_color}{gray}{0}\pgfsetstrokecolor{eps2pgf_color}\pgfsetfillcolor{eps2pgf_color}
\pgfpathmoveto{\pgfqpoint{0.273cm}{1.789cm}}
\pgfpathcurveto{\pgfqpoint{0.273cm}{1.825cm}}{\pgfqpoint{0.259cm}{1.86cm}}{\pgfqpoint{0.233cm}{1.886cm}}
\pgfpathcurveto{\pgfqpoint{0.207cm}{1.912cm}}{\pgfqpoint{0.173cm}{1.926cm}}{\pgfqpoint{0.137cm}{1.926cm}}
\pgfpathcurveto{\pgfqpoint{0.1cm}{1.926cm}}{\pgfqpoint{0.066cm}{1.912cm}}{\pgfqpoint{0.04cm}{1.886cm}}
\pgfpathcurveto{\pgfqpoint{0.014cm}{1.86cm}}{\pgfqpoint{0cm}{1.825cm}}{\pgfqpoint{0cm}{1.789cm}}
\pgfpathcurveto{\pgfqpoint{0cm}{1.753cm}}{\pgfqpoint{0.014cm}{1.718cm}}{\pgfqpoint{0.04cm}{1.692cm}}
\pgfpathcurveto{\pgfqpoint{0.066cm}{1.667cm}}{\pgfqpoint{0.1cm}{1.652cm}}{\pgfqpoint{0.137cm}{1.652cm}}
\pgfpathcurveto{\pgfqpoint{0.173cm}{1.652cm}}{\pgfqpoint{0.207cm}{1.667cm}}{\pgfqpoint{0.233cm}{1.692cm}}
\pgfpathcurveto{\pgfqpoint{0.259cm}{1.718cm}}{\pgfqpoint{0.273cm}{1.753cm}}{\pgfqpoint{0.273cm}{1.789cm}}
\pgfusepath{fill}
\pgfpathmoveto{\pgfqpoint{1.345cm}{1.765cm}}
\pgfpathcurveto{\pgfqpoint{1.345cm}{1.801cm}}{\pgfqpoint{1.331cm}{1.836cm}}{\pgfqpoint{1.305cm}{1.862cm}}
\pgfpathcurveto{\pgfqpoint{1.28cm}{1.887cm}}{\pgfqpoint{1.245cm}{1.902cm}}{\pgfqpoint{1.209cm}{1.902cm}}
\pgfpathcurveto{\pgfqpoint{1.172cm}{1.902cm}}{\pgfqpoint{1.138cm}{1.887cm}}{\pgfqpoint{1.112cm}{1.862cm}}
\pgfpathcurveto{\pgfqpoint{1.087cm}{1.836cm}}{\pgfqpoint{1.072cm}{1.801cm}}{\pgfqpoint{1.072cm}{1.765cm}}
\pgfpathcurveto{\pgfqpoint{1.072cm}{1.728cm}}{\pgfqpoint{1.087cm}{1.694cm}}{\pgfqpoint{1.112cm}{1.668cm}}
\pgfpathcurveto{\pgfqpoint{1.138cm}{1.642cm}}{\pgfqpoint{1.172cm}{1.628cm}}{\pgfqpoint{1.209cm}{1.628cm}}
\pgfpathcurveto{\pgfqpoint{1.245cm}{1.628cm}}{\pgfqpoint{1.28cm}{1.642cm}}{\pgfqpoint{1.305cm}{1.668cm}}
\pgfpathcurveto{\pgfqpoint{1.331cm}{1.694cm}}{\pgfqpoint{1.345cm}{1.728cm}}{\pgfqpoint{1.345cm}{1.765cm}}
\pgfusepath{fill}
\begin{pgfscope}
\pgfsetdash{}{0cm}
\pgfsetlinewidth{0.818mm}
\pgfsetroundcap
\pgfsetroundjoin
\pgfsetmiterlimit{7.0}
\pgfpathmoveto{\pgfqpoint{0.682cm}{1.065cm}}
\pgfpathlineto{\pgfqpoint{1.246cm}{0.315cm}}
\pgfpathlineto{\pgfqpoint{1.811cm}{1.065cm}}
\pgfusepath{stroke}
\end{pgfscope}
\pgfpathmoveto{\pgfqpoint{1.948cm}{1.065cm}}
\pgfpathcurveto{\pgfqpoint{1.948cm}{1.101cm}}{\pgfqpoint{1.933cm}{1.136cm}}{\pgfqpoint{1.907cm}{1.162cm}}
\pgfpathcurveto{\pgfqpoint{1.882cm}{1.187cm}}{\pgfqpoint{1.847cm}{1.202cm}}{\pgfqpoint{1.811cm}{1.202cm}}
\pgfpathcurveto{\pgfqpoint{1.775cm}{1.202cm}}{\pgfqpoint{1.74cm}{1.187cm}}{\pgfqpoint{1.714cm}{1.162cm}}
\pgfpathcurveto{\pgfqpoint{1.689cm}{1.136cm}}{\pgfqpoint{1.674cm}{1.101cm}}{\pgfqpoint{1.674cm}{1.065cm}}
\pgfpathcurveto{\pgfqpoint{1.674cm}{1.029cm}}{\pgfqpoint{1.689cm}{0.994cm}}{\pgfqpoint{1.714cm}{0.968cm}}
\pgfpathcurveto{\pgfqpoint{1.74cm}{0.942cm}}{\pgfqpoint{1.775cm}{0.928cm}}{\pgfqpoint{1.811cm}{0.928cm}}
\pgfpathcurveto{\pgfqpoint{1.847cm}{0.928cm}}{\pgfqpoint{1.882cm}{0.942cm}}{\pgfqpoint{1.907cm}{0.968cm}}
\pgfpathcurveto{\pgfqpoint{1.933cm}{0.994cm}}{\pgfqpoint{1.948cm}{1.029cm}}{\pgfqpoint{1.948cm}{1.065cm}}
\pgfusepath{fill}
\begin{pgfscope}
\pgfsetdash{}{0cm}
\pgfsetlinewidth{0.818mm}
\pgfsetmiterlimit{7.0}
\pgfpathmoveto{\pgfqpoint{1.246cm}{0.315cm}}
\pgfpathlineto{\pgfqpoint{1.244cm}{1.061cm}}
\pgfusepath{stroke}
\end{pgfscope}
\pgfpathmoveto{\pgfqpoint{1.38cm}{1.065cm}}
\pgfpathcurveto{\pgfqpoint{1.38cm}{1.101cm}}{\pgfqpoint{1.366cm}{1.136cm}}{\pgfqpoint{1.34cm}{1.162cm}}
\pgfpathcurveto{\pgfqpoint{1.315cm}{1.187cm}}{\pgfqpoint{1.28cm}{1.202cm}}{\pgfqpoint{1.244cm}{1.202cm}}
\pgfpathcurveto{\pgfqpoint{1.207cm}{1.202cm}}{\pgfqpoint{1.173cm}{1.187cm}}{\pgfqpoint{1.147cm}{1.162cm}}
\pgfpathcurveto{\pgfqpoint{1.121cm}{1.136cm}}{\pgfqpoint{1.107cm}{1.101cm}}{\pgfqpoint{1.107cm}{1.065cm}}
\pgfpathcurveto{\pgfqpoint{1.107cm}{1.029cm}}{\pgfqpoint{1.121cm}{0.994cm}}{\pgfqpoint{1.147cm}{0.968cm}}
\pgfpathcurveto{\pgfqpoint{1.173cm}{0.942cm}}{\pgfqpoint{1.207cm}{0.928cm}}{\pgfqpoint{1.244cm}{0.928cm}}
\pgfpathcurveto{\pgfqpoint{1.28cm}{0.928cm}}{\pgfqpoint{1.315cm}{0.942cm}}{\pgfqpoint{1.34cm}{0.968cm}}
\pgfpathcurveto{\pgfqpoint{1.366cm}{0.994cm}}{\pgfqpoint{1.38cm}{1.029cm}}{\pgfqpoint{1.38cm}{1.065cm}}
\pgfusepath{fill}
\begin{pgfscope}
\pgfsetdash{}{0cm}
\pgfsetlinewidth{0.818mm}
\pgfsetmiterlimit{4.0}
\pgfpathmoveto{\pgfqpoint{1.383cm}{0.178cm}}
\pgfpathcurveto{\pgfqpoint{1.383cm}{0.214cm}}{\pgfqpoint{1.369cm}{0.249cm}}{\pgfqpoint{1.343cm}{0.275cm}}
\pgfpathcurveto{\pgfqpoint{1.317cm}{0.3cm}}{\pgfqpoint{1.283cm}{0.315cm}}{\pgfqpoint{1.246cm}{0.315cm}}
\pgfpathcurveto{\pgfqpoint{1.21cm}{0.315cm}}{\pgfqpoint{1.175cm}{0.3cm}}{\pgfqpoint{1.15cm}{0.275cm}}
\pgfpathcurveto{\pgfqpoint{1.124cm}{0.249cm}}{\pgfqpoint{1.11cm}{0.214cm}}{\pgfqpoint{1.11cm}{0.178cm}}
\pgfpathcurveto{\pgfqpoint{1.11cm}{0.141cm}}{\pgfqpoint{1.124cm}{0.107cm}}{\pgfqpoint{1.15cm}{0.081cm}}
\pgfpathcurveto{\pgfqpoint{1.175cm}{0.055cm}}{\pgfqpoint{1.21cm}{0.041cm}}{\pgfqpoint{1.246cm}{0.041cm}}
\pgfpathcurveto{\pgfqpoint{1.283cm}{0.041cm}}{\pgfqpoint{1.317cm}{0.055cm}}{\pgfqpoint{1.343cm}{0.081cm}}
\pgfpathcurveto{\pgfqpoint{1.369cm}{0.107cm}}{\pgfqpoint{1.383cm}{0.141cm}}{\pgfqpoint{1.383cm}{0.178cm}}
\pgfusepath{stroke}
\end{pgfscope}
\end{pgfscope}
\end{pgfscope}
\end{pgfscope}
\end{tikzpicture}}} = X^{\!\resizebox{0.6em}{!}{
\begin{tikzpicture}
\pgfpathmoveto{\pgfqpoint{0cm}{0cm}}
\pgfpathlineto{\pgfqpoint{1.376cm}{0cm}}
\pgfpathlineto{\pgfqpoint{1.376cm}{1.588cm}}
\pgfpathlineto{\pgfqpoint{0cm}{1.588cm}}
\pgfpathclose
\pgfusepath{clip}
\begin{pgfscope}
\begin{pgfscope}
\pgfpathmoveto{\pgfqpoint{0cm}{0cm}}
\pgfpathlineto{\pgfqpoint{1.376cm}{0cm}}
\pgfpathlineto{\pgfqpoint{1.376cm}{1.588cm}}
\pgfpathlineto{\pgfqpoint{0cm}{1.588cm}}
\pgfpathclose
\pgfusepath{clip}
\begin{pgfscope}
\begin{pgfscope}
\definecolor{eps2pgf_color}{gray}{0.976471}\pgfsetstrokecolor{eps2pgf_color}\pgfsetfillcolor{eps2pgf_color}
\pgfpathmoveto{\pgfqpoint{0cm}{0cm}}
\pgfpathlineto{\pgfqpoint{1.376cm}{0cm}}
\pgfpathlineto{\pgfqpoint{1.376cm}{1.588cm}}
\pgfpathlineto{\pgfqpoint{0cm}{1.588cm}}
\pgfpathclose
\pgfusepath{fill}
\end{pgfscope}
\begin{pgfscope}
\pgfsetdash{}{0cm}
\pgfsetlinewidth{0.818mm}
\pgfsetroundcap
\pgfsetroundjoin
\pgfsetmiterlimit{7.0}
\definecolor{eps2pgf_color}{gray}{0}\pgfsetstrokecolor{eps2pgf_color}\pgfsetfillcolor{eps2pgf_color}
\pgfpathmoveto{\pgfqpoint{0.117cm}{1.476cm}}
\pgfpathlineto{\pgfqpoint{0.682cm}{0.726cm}}
\pgfpathlineto{\pgfqpoint{1.246cm}{1.476cm}}
\pgfusepath{stroke}
\end{pgfscope}
\definecolor{eps2pgf_color}{gray}{0}\pgfsetstrokecolor{eps2pgf_color}\pgfsetfillcolor{eps2pgf_color}
\pgfpathmoveto{\pgfqpoint{0.273cm}{1.451cm}}
\pgfpathcurveto{\pgfqpoint{0.273cm}{1.487cm}}{\pgfqpoint{0.259cm}{1.522cm}}{\pgfqpoint{0.233cm}{1.547cm}}
\pgfpathcurveto{\pgfqpoint{0.207cm}{1.573cm}}{\pgfqpoint{0.173cm}{1.588cm}}{\pgfqpoint{0.137cm}{1.588cm}}
\pgfpathcurveto{\pgfqpoint{0.1cm}{1.588cm}}{\pgfqpoint{0.066cm}{1.573cm}}{\pgfqpoint{0.04cm}{1.547cm}}
\pgfpathcurveto{\pgfqpoint{0.014cm}{1.522cm}}{\pgfqpoint{0cm}{1.487cm}}{\pgfqpoint{0cm}{1.451cm}}
\pgfpathcurveto{\pgfqpoint{0cm}{1.414cm}}{\pgfqpoint{0.014cm}{1.379cm}}{\pgfqpoint{0.04cm}{1.354cm}}
\pgfpathcurveto{\pgfqpoint{0.066cm}{1.328cm}}{\pgfqpoint{0.1cm}{1.314cm}}{\pgfqpoint{0.137cm}{1.314cm}}
\pgfpathcurveto{\pgfqpoint{0.173cm}{1.314cm}}{\pgfqpoint{0.207cm}{1.328cm}}{\pgfqpoint{0.233cm}{1.354cm}}
\pgfpathcurveto{\pgfqpoint{0.259cm}{1.379cm}}{\pgfqpoint{0.273cm}{1.414cm}}{\pgfqpoint{0.273cm}{1.451cm}}
\pgfusepath{fill}
\pgfpathmoveto{\pgfqpoint{1.345cm}{1.426cm}}
\pgfpathcurveto{\pgfqpoint{1.345cm}{1.463cm}}{\pgfqpoint{1.331cm}{1.497cm}}{\pgfqpoint{1.305cm}{1.523cm}}
\pgfpathcurveto{\pgfqpoint{1.28cm}{1.549cm}}{\pgfqpoint{1.245cm}{1.563cm}}{\pgfqpoint{1.209cm}{1.563cm}}
\pgfpathcurveto{\pgfqpoint{1.172cm}{1.563cm}}{\pgfqpoint{1.138cm}{1.549cm}}{\pgfqpoint{1.112cm}{1.523cm}}
\pgfpathcurveto{\pgfqpoint{1.087cm}{1.497cm}}{\pgfqpoint{1.072cm}{1.463cm}}{\pgfqpoint{1.072cm}{1.426cm}}
\pgfpathcurveto{\pgfqpoint{1.072cm}{1.39cm}}{\pgfqpoint{1.087cm}{1.355cm}}{\pgfqpoint{1.112cm}{1.329cm}}
\pgfpathcurveto{\pgfqpoint{1.138cm}{1.304cm}}{\pgfqpoint{1.172cm}{1.289cm}}{\pgfqpoint{1.209cm}{1.289cm}}
\pgfpathcurveto{\pgfqpoint{1.245cm}{1.289cm}}{\pgfqpoint{1.28cm}{1.304cm}}{\pgfqpoint{1.305cm}{1.329cm}}
\pgfpathcurveto{\pgfqpoint{1.331cm}{1.355cm}}{\pgfqpoint{1.345cm}{1.39cm}}{\pgfqpoint{1.345cm}{1.426cm}}
\pgfusepath{fill}
\begin{pgfscope}
\pgfsetdash{}{0cm}
\pgfsetlinewidth{0.818mm}
\pgfsetroundcap
\pgfsetmiterlimit{4.0}
\pgfpathmoveto{\pgfqpoint{0.682cm}{0.726cm}}
\pgfpathlineto{\pgfqpoint{0.682cm}{0.097cm}}
\pgfusepath{stroke}
\end{pgfscope}
\end{pgfscope}
\end{pgfscope}
\end{pgfscope}
\end{tikzpicture}}} \circ \llbracket X^2 \rrbracket -\frac{b}{3}, \qquad
   X^{\!\resizebox{!}{.8em}{
\begin{tikzpicture}
\pgfpathmoveto{\pgfqpoint{0cm}{-0.035cm}}
\pgfpathlineto{\pgfqpoint{1.976cm}{-0.035cm}}
\pgfpathlineto{\pgfqpoint{1.976cm}{1.94cm}}
\pgfpathlineto{\pgfqpoint{0cm}{1.94cm}}
\pgfpathclose
\pgfusepath{clip}
\begin{pgfscope}
\begin{pgfscope}
\pgfpathmoveto{\pgfqpoint{0cm}{-0.035cm}}
\pgfpathlineto{\pgfqpoint{1.976cm}{-0.035cm}}
\pgfpathlineto{\pgfqpoint{1.976cm}{1.94cm}}
\pgfpathlineto{\pgfqpoint{0cm}{1.94cm}}
\pgfpathclose
\pgfusepath{clip}
\begin{pgfscope}
\begin{pgfscope}
\pgfsetdash{}{0cm}
\pgfsetlinewidth{0.818mm}
\pgfsetroundcap
\pgfsetroundjoin
\pgfsetmiterlimit{7.0}
\definecolor{eps2pgf_color}{gray}{0}\pgfsetstrokecolor{eps2pgf_color}\pgfsetfillcolor{eps2pgf_color}
\pgfpathmoveto{\pgfqpoint{0.117cm}{1.815cm}}
\pgfpathlineto{\pgfqpoint{0.682cm}{1.065cm}}
\pgfpathlineto{\pgfqpoint{1.246cm}{1.815cm}}
\pgfusepath{stroke}
\end{pgfscope}
\definecolor{eps2pgf_color}{gray}{0}\pgfsetstrokecolor{eps2pgf_color}\pgfsetfillcolor{eps2pgf_color}
\pgfpathmoveto{\pgfqpoint{0.273cm}{1.789cm}}
\pgfpathcurveto{\pgfqpoint{0.273cm}{1.825cm}}{\pgfqpoint{0.259cm}{1.86cm}}{\pgfqpoint{0.233cm}{1.886cm}}
\pgfpathcurveto{\pgfqpoint{0.207cm}{1.912cm}}{\pgfqpoint{0.173cm}{1.926cm}}{\pgfqpoint{0.137cm}{1.926cm}}
\pgfpathcurveto{\pgfqpoint{0.1cm}{1.926cm}}{\pgfqpoint{0.066cm}{1.912cm}}{\pgfqpoint{0.04cm}{1.886cm}}
\pgfpathcurveto{\pgfqpoint{0.014cm}{1.86cm}}{\pgfqpoint{0cm}{1.825cm}}{\pgfqpoint{0cm}{1.789cm}}
\pgfpathcurveto{\pgfqpoint{0cm}{1.753cm}}{\pgfqpoint{0.014cm}{1.718cm}}{\pgfqpoint{0.04cm}{1.692cm}}
\pgfpathcurveto{\pgfqpoint{0.066cm}{1.667cm}}{\pgfqpoint{0.1cm}{1.652cm}}{\pgfqpoint{0.137cm}{1.652cm}}
\pgfpathcurveto{\pgfqpoint{0.173cm}{1.652cm}}{\pgfqpoint{0.207cm}{1.667cm}}{\pgfqpoint{0.233cm}{1.692cm}}
\pgfpathcurveto{\pgfqpoint{0.259cm}{1.718cm}}{\pgfqpoint{0.273cm}{1.753cm}}{\pgfqpoint{0.273cm}{1.789cm}}
\pgfusepath{fill}
\begin{pgfscope}
\pgfsetdash{}{0cm}
\pgfsetlinewidth{0.818mm}
\pgfsetmiterlimit{7.0}
\pgfpathmoveto{\pgfqpoint{0.682cm}{1.065cm}}
\pgfpathlineto{\pgfqpoint{0.679cm}{1.812cm}}
\pgfusepath{stroke}
\end{pgfscope}
\pgfpathmoveto{\pgfqpoint{0.815cm}{1.793cm}}
\pgfpathcurveto{\pgfqpoint{0.815cm}{1.829cm}}{\pgfqpoint{0.801cm}{1.864cm}}{\pgfqpoint{0.775cm}{1.89cm}}
\pgfpathcurveto{\pgfqpoint{0.75cm}{1.915cm}}{\pgfqpoint{0.715cm}{1.93cm}}{\pgfqpoint{0.679cm}{1.93cm}}
\pgfpathcurveto{\pgfqpoint{0.643cm}{1.93cm}}{\pgfqpoint{0.608cm}{1.915cm}}{\pgfqpoint{0.582cm}{1.89cm}}
\pgfpathcurveto{\pgfqpoint{0.557cm}{1.864cm}}{\pgfqpoint{0.542cm}{1.829cm}}{\pgfqpoint{0.542cm}{1.793cm}}
\pgfpathcurveto{\pgfqpoint{0.542cm}{1.756cm}}{\pgfqpoint{0.557cm}{1.722cm}}{\pgfqpoint{0.582cm}{1.696cm}}
\pgfpathcurveto{\pgfqpoint{0.608cm}{1.67cm}}{\pgfqpoint{0.643cm}{1.656cm}}{\pgfqpoint{0.679cm}{1.656cm}}
\pgfpathcurveto{\pgfqpoint{0.715cm}{1.656cm}}{\pgfqpoint{0.75cm}{1.67cm}}{\pgfqpoint{0.775cm}{1.696cm}}
\pgfpathcurveto{\pgfqpoint{0.801cm}{1.722cm}}{\pgfqpoint{0.815cm}{1.756cm}}{\pgfqpoint{0.815cm}{1.793cm}}
\pgfusepath{fill}
\pgfpathmoveto{\pgfqpoint{1.345cm}{1.765cm}}
\pgfpathcurveto{\pgfqpoint{1.345cm}{1.801cm}}{\pgfqpoint{1.331cm}{1.836cm}}{\pgfqpoint{1.305cm}{1.862cm}}
\pgfpathcurveto{\pgfqpoint{1.28cm}{1.887cm}}{\pgfqpoint{1.245cm}{1.902cm}}{\pgfqpoint{1.209cm}{1.902cm}}
\pgfpathcurveto{\pgfqpoint{1.172cm}{1.902cm}}{\pgfqpoint{1.138cm}{1.887cm}}{\pgfqpoint{1.112cm}{1.862cm}}
\pgfpathcurveto{\pgfqpoint{1.087cm}{1.836cm}}{\pgfqpoint{1.072cm}{1.801cm}}{\pgfqpoint{1.072cm}{1.765cm}}
\pgfpathcurveto{\pgfqpoint{1.072cm}{1.728cm}}{\pgfqpoint{1.087cm}{1.694cm}}{\pgfqpoint{1.112cm}{1.668cm}}
\pgfpathcurveto{\pgfqpoint{1.138cm}{1.642cm}}{\pgfqpoint{1.172cm}{1.628cm}}{\pgfqpoint{1.209cm}{1.628cm}}
\pgfpathcurveto{\pgfqpoint{1.245cm}{1.628cm}}{\pgfqpoint{1.28cm}{1.642cm}}{\pgfqpoint{1.305cm}{1.668cm}}
\pgfpathcurveto{\pgfqpoint{1.331cm}{1.694cm}}{\pgfqpoint{1.345cm}{1.728cm}}{\pgfqpoint{1.345cm}{1.765cm}}
\pgfusepath{fill}
\begin{pgfscope}
\pgfsetdash{}{0cm}
\pgfsetlinewidth{0.818mm}
\pgfsetroundcap
\pgfsetroundjoin
\pgfsetmiterlimit{7.0}
\pgfpathmoveto{\pgfqpoint{0.682cm}{1.065cm}}
\pgfpathlineto{\pgfqpoint{1.246cm}{0.315cm}}
\pgfpathlineto{\pgfqpoint{1.811cm}{1.065cm}}
\pgfusepath{stroke}
\end{pgfscope}
\pgfpathmoveto{\pgfqpoint{1.948cm}{1.065cm}}
\pgfpathcurveto{\pgfqpoint{1.948cm}{1.101cm}}{\pgfqpoint{1.933cm}{1.136cm}}{\pgfqpoint{1.907cm}{1.162cm}}
\pgfpathcurveto{\pgfqpoint{1.882cm}{1.187cm}}{\pgfqpoint{1.847cm}{1.202cm}}{\pgfqpoint{1.811cm}{1.202cm}}
\pgfpathcurveto{\pgfqpoint{1.775cm}{1.202cm}}{\pgfqpoint{1.74cm}{1.187cm}}{\pgfqpoint{1.714cm}{1.162cm}}
\pgfpathcurveto{\pgfqpoint{1.689cm}{1.136cm}}{\pgfqpoint{1.674cm}{1.101cm}}{\pgfqpoint{1.674cm}{1.065cm}}
\pgfpathcurveto{\pgfqpoint{1.674cm}{1.029cm}}{\pgfqpoint{1.689cm}{0.994cm}}{\pgfqpoint{1.714cm}{0.968cm}}
\pgfpathcurveto{\pgfqpoint{1.74cm}{0.942cm}}{\pgfqpoint{1.775cm}{0.928cm}}{\pgfqpoint{1.811cm}{0.928cm}}
\pgfpathcurveto{\pgfqpoint{1.847cm}{0.928cm}}{\pgfqpoint{1.882cm}{0.942cm}}{\pgfqpoint{1.907cm}{0.968cm}}
\pgfpathcurveto{\pgfqpoint{1.933cm}{0.994cm}}{\pgfqpoint{1.948cm}{1.029cm}}{\pgfqpoint{1.948cm}{1.065cm}}
\pgfusepath{fill}
\begin{pgfscope}
\pgfsetdash{}{0cm}
\pgfsetlinewidth{0.818mm}
\pgfsetmiterlimit{7.0}
\pgfpathmoveto{\pgfqpoint{1.246cm}{0.315cm}}
\pgfpathlineto{\pgfqpoint{1.244cm}{1.061cm}}
\pgfusepath{stroke}
\end{pgfscope}
\pgfpathmoveto{\pgfqpoint{1.38cm}{1.065cm}}
\pgfpathcurveto{\pgfqpoint{1.38cm}{1.101cm}}{\pgfqpoint{1.366cm}{1.136cm}}{\pgfqpoint{1.34cm}{1.162cm}}
\pgfpathcurveto{\pgfqpoint{1.315cm}{1.187cm}}{\pgfqpoint{1.28cm}{1.202cm}}{\pgfqpoint{1.244cm}{1.202cm}}
\pgfpathcurveto{\pgfqpoint{1.207cm}{1.202cm}}{\pgfqpoint{1.173cm}{1.187cm}}{\pgfqpoint{1.147cm}{1.162cm}}
\pgfpathcurveto{\pgfqpoint{1.121cm}{1.136cm}}{\pgfqpoint{1.107cm}{1.101cm}}{\pgfqpoint{1.107cm}{1.065cm}}
\pgfpathcurveto{\pgfqpoint{1.107cm}{1.029cm}}{\pgfqpoint{1.121cm}{0.994cm}}{\pgfqpoint{1.147cm}{0.968cm}}
\pgfpathcurveto{\pgfqpoint{1.173cm}{0.942cm}}{\pgfqpoint{1.207cm}{0.928cm}}{\pgfqpoint{1.244cm}{0.928cm}}
\pgfpathcurveto{\pgfqpoint{1.28cm}{0.928cm}}{\pgfqpoint{1.315cm}{0.942cm}}{\pgfqpoint{1.34cm}{0.968cm}}
\pgfpathcurveto{\pgfqpoint{1.366cm}{0.994cm}}{\pgfqpoint{1.38cm}{1.029cm}}{\pgfqpoint{1.38cm}{1.065cm}}
\pgfusepath{fill}
\begin{pgfscope}
\pgfsetdash{}{0cm}
\pgfsetlinewidth{0.818mm}
\pgfsetmiterlimit{4.0}
\pgfpathmoveto{\pgfqpoint{1.383cm}{0.178cm}}
\pgfpathcurveto{\pgfqpoint{1.383cm}{0.214cm}}{\pgfqpoint{1.369cm}{0.249cm}}{\pgfqpoint{1.343cm}{0.275cm}}
\pgfpathcurveto{\pgfqpoint{1.317cm}{0.3cm}}{\pgfqpoint{1.283cm}{0.315cm}}{\pgfqpoint{1.246cm}{0.315cm}}
\pgfpathcurveto{\pgfqpoint{1.21cm}{0.315cm}}{\pgfqpoint{1.175cm}{0.3cm}}{\pgfqpoint{1.15cm}{0.275cm}}
\pgfpathcurveto{\pgfqpoint{1.124cm}{0.249cm}}{\pgfqpoint{1.11cm}{0.214cm}}{\pgfqpoint{1.11cm}{0.178cm}}
\pgfpathcurveto{\pgfqpoint{1.11cm}{0.141cm}}{\pgfqpoint{1.124cm}{0.107cm}}{\pgfqpoint{1.15cm}{0.081cm}}
\pgfpathcurveto{\pgfqpoint{1.175cm}{0.055cm}}{\pgfqpoint{1.21cm}{0.041cm}}{\pgfqpoint{1.246cm}{0.041cm}}
\pgfpathcurveto{\pgfqpoint{1.283cm}{0.041cm}}{\pgfqpoint{1.317cm}{0.055cm}}{\pgfqpoint{1.343cm}{0.081cm}}
\pgfpathcurveto{\pgfqpoint{1.369cm}{0.107cm}}{\pgfqpoint{1.383cm}{0.141cm}}{\pgfqpoint{1.383cm}{0.178cm}}
\pgfusepath{stroke}
\end{pgfscope}
\end{pgfscope}
\end{pgfscope}
\end{pgfscope}
\end{tikzpicture}}} = X^{\!\resizebox{0.6em}{!}{
\begin{tikzpicture}
\pgfpathmoveto{\pgfqpoint{0cm}{-0.035cm}}
\pgfpathlineto{\pgfqpoint{1.376cm}{-0.035cm}}
\pgfpathlineto{\pgfqpoint{1.376cm}{1.552cm}}
\pgfpathlineto{\pgfqpoint{0cm}{1.552cm}}
\pgfpathclose
\pgfusepath{clip}
\begin{pgfscope}
\begin{pgfscope}
\pgfpathmoveto{\pgfqpoint{0cm}{-0.035cm}}
\pgfpathlineto{\pgfqpoint{1.376cm}{-0.035cm}}
\pgfpathlineto{\pgfqpoint{1.376cm}{1.552cm}}
\pgfpathlineto{\pgfqpoint{0cm}{1.552cm}}
\pgfpathclose
\pgfusepath{clip}
\begin{pgfscope}
\begin{pgfscope}
\pgfsetdash{}{0cm}
\pgfsetlinewidth{0.818mm}
\pgfsetroundcap
\pgfsetroundjoin
\pgfsetmiterlimit{7.0}
\definecolor{eps2pgf_color}{gray}{0}\pgfsetstrokecolor{eps2pgf_color}\pgfsetfillcolor{eps2pgf_color}
\pgfpathmoveto{\pgfqpoint{0.117cm}{1.421cm}}
\pgfpathlineto{\pgfqpoint{0.682cm}{0.671cm}}
\pgfpathlineto{\pgfqpoint{1.246cm}{1.421cm}}
\pgfusepath{stroke}
\end{pgfscope}
\definecolor{eps2pgf_color}{gray}{0}\pgfsetstrokecolor{eps2pgf_color}\pgfsetfillcolor{eps2pgf_color}
\pgfpathmoveto{\pgfqpoint{0.273cm}{1.395cm}}
\pgfpathcurveto{\pgfqpoint{0.273cm}{1.432cm}}{\pgfqpoint{0.259cm}{1.467cm}}{\pgfqpoint{0.233cm}{1.492cm}}
\pgfpathcurveto{\pgfqpoint{0.207cm}{1.518cm}}{\pgfqpoint{0.173cm}{1.532cm}}{\pgfqpoint{0.137cm}{1.532cm}}
\pgfpathcurveto{\pgfqpoint{0.1cm}{1.532cm}}{\pgfqpoint{0.066cm}{1.518cm}}{\pgfqpoint{0.04cm}{1.492cm}}
\pgfpathcurveto{\pgfqpoint{0.014cm}{1.467cm}}{\pgfqpoint{0cm}{1.432cm}}{\pgfqpoint{0cm}{1.395cm}}
\pgfpathcurveto{\pgfqpoint{0cm}{1.359cm}}{\pgfqpoint{0.014cm}{1.324cm}}{\pgfqpoint{0.04cm}{1.299cm}}
\pgfpathcurveto{\pgfqpoint{0.066cm}{1.273cm}}{\pgfqpoint{0.1cm}{1.258cm}}{\pgfqpoint{0.137cm}{1.258cm}}
\pgfpathcurveto{\pgfqpoint{0.173cm}{1.258cm}}{\pgfqpoint{0.207cm}{1.273cm}}{\pgfqpoint{0.233cm}{1.299cm}}
\pgfpathcurveto{\pgfqpoint{0.259cm}{1.324cm}}{\pgfqpoint{0.273cm}{1.359cm}}{\pgfqpoint{0.273cm}{1.395cm}}
\pgfusepath{fill}
\begin{pgfscope}
\pgfsetdash{}{0cm}
\pgfsetlinewidth{0.818mm}
\pgfsetmiterlimit{7.0}
\pgfpathmoveto{\pgfqpoint{0.682cm}{0.671cm}}
\pgfpathlineto{\pgfqpoint{0.679cm}{1.418cm}}
\pgfusepath{stroke}
\end{pgfscope}
\pgfpathmoveto{\pgfqpoint{0.815cm}{1.399cm}}
\pgfpathcurveto{\pgfqpoint{0.815cm}{1.435cm}}{\pgfqpoint{0.801cm}{1.47cm}}{\pgfqpoint{0.775cm}{1.496cm}}
\pgfpathcurveto{\pgfqpoint{0.75cm}{1.521cm}}{\pgfqpoint{0.715cm}{1.536cm}}{\pgfqpoint{0.679cm}{1.536cm}}
\pgfpathcurveto{\pgfqpoint{0.643cm}{1.536cm}}{\pgfqpoint{0.608cm}{1.521cm}}{\pgfqpoint{0.582cm}{1.496cm}}
\pgfpathcurveto{\pgfqpoint{0.557cm}{1.47cm}}{\pgfqpoint{0.542cm}{1.435cm}}{\pgfqpoint{0.542cm}{1.399cm}}
\pgfpathcurveto{\pgfqpoint{0.542cm}{1.363cm}}{\pgfqpoint{0.557cm}{1.328cm}}{\pgfqpoint{0.582cm}{1.302cm}}
\pgfpathcurveto{\pgfqpoint{0.608cm}{1.276cm}}{\pgfqpoint{0.643cm}{1.262cm}}{\pgfqpoint{0.679cm}{1.262cm}}
\pgfpathcurveto{\pgfqpoint{0.715cm}{1.262cm}}{\pgfqpoint{0.75cm}{1.276cm}}{\pgfqpoint{0.775cm}{1.302cm}}
\pgfpathcurveto{\pgfqpoint{0.801cm}{1.328cm}}{\pgfqpoint{0.815cm}{1.363cm}}{\pgfqpoint{0.815cm}{1.399cm}}
\pgfusepath{fill}
\pgfpathmoveto{\pgfqpoint{1.345cm}{1.371cm}}
\pgfpathcurveto{\pgfqpoint{1.345cm}{1.408cm}}{\pgfqpoint{1.331cm}{1.442cm}}{\pgfqpoint{1.305cm}{1.468cm}}
\pgfpathcurveto{\pgfqpoint{1.28cm}{1.494cm}}{\pgfqpoint{1.245cm}{1.508cm}}{\pgfqpoint{1.209cm}{1.508cm}}
\pgfpathcurveto{\pgfqpoint{1.172cm}{1.508cm}}{\pgfqpoint{1.138cm}{1.494cm}}{\pgfqpoint{1.112cm}{1.468cm}}
\pgfpathcurveto{\pgfqpoint{1.087cm}{1.442cm}}{\pgfqpoint{1.072cm}{1.408cm}}{\pgfqpoint{1.072cm}{1.371cm}}
\pgfpathcurveto{\pgfqpoint{1.072cm}{1.335cm}}{\pgfqpoint{1.087cm}{1.3cm}}{\pgfqpoint{1.112cm}{1.274cm}}
\pgfpathcurveto{\pgfqpoint{1.138cm}{1.249cm}}{\pgfqpoint{1.172cm}{1.234cm}}{\pgfqpoint{1.209cm}{1.234cm}}
\pgfpathcurveto{\pgfqpoint{1.245cm}{1.234cm}}{\pgfqpoint{1.28cm}{1.249cm}}{\pgfqpoint{1.305cm}{1.274cm}}
\pgfpathcurveto{\pgfqpoint{1.331cm}{1.3cm}}{\pgfqpoint{1.345cm}{1.335cm}}{\pgfqpoint{1.345cm}{1.371cm}}
\pgfusepath{fill}
\begin{pgfscope}
\pgfsetdash{}{0cm}
\pgfsetlinewidth{0.818mm}
\pgfsetroundcap
\pgfsetmiterlimit{4.0}
\pgfpathmoveto{\pgfqpoint{0.682cm}{0.671cm}}
\pgfpathlineto{\pgfqpoint{0.682cm}{0.042cm}}
\pgfusepath{stroke}
\end{pgfscope}
\end{pgfscope}
\end{pgfscope}
\end{pgfscope}
\end{tikzpicture}}} \circ \llbracket X^2 \rrbracket -  bX. \]
Thus, the left hand side of \eqref{eq:43} rewrites as
\begin{align*}
\begin{aligned}
&\LL \varphi + \varphi^3 + (- 3 a + 3b) \varphi - \xi \\
&\quad= \LL\phi+\LL\psi +3\llbracket X^2 \rrbracket(-X^{\!\resizebox{0.6em}{!}{
\begin{tikzpicture}
\pgfpathmoveto{\pgfqpoint{0cm}{-0.035cm}}
\pgfpathlineto{\pgfqpoint{1.376cm}{-0.035cm}}
\pgfpathlineto{\pgfqpoint{1.376cm}{1.552cm}}
\pgfpathlineto{\pgfqpoint{0cm}{1.552cm}}
\pgfpathclose
\pgfusepath{clip}
\begin{pgfscope}
\begin{pgfscope}
\pgfpathmoveto{\pgfqpoint{0cm}{-0.035cm}}
\pgfpathlineto{\pgfqpoint{1.376cm}{-0.035cm}}
\pgfpathlineto{\pgfqpoint{1.376cm}{1.552cm}}
\pgfpathlineto{\pgfqpoint{0cm}{1.552cm}}
\pgfpathclose
\pgfusepath{clip}
\begin{pgfscope}
\begin{pgfscope}
\pgfsetdash{}{0cm}
\pgfsetlinewidth{0.818mm}
\pgfsetroundcap
\pgfsetroundjoin
\pgfsetmiterlimit{7.0}
\definecolor{eps2pgf_color}{gray}{0}\pgfsetstrokecolor{eps2pgf_color}\pgfsetfillcolor{eps2pgf_color}
\pgfpathmoveto{\pgfqpoint{0.117cm}{1.421cm}}
\pgfpathlineto{\pgfqpoint{0.682cm}{0.671cm}}
\pgfpathlineto{\pgfqpoint{1.246cm}{1.421cm}}
\pgfusepath{stroke}
\end{pgfscope}
\definecolor{eps2pgf_color}{gray}{0}\pgfsetstrokecolor{eps2pgf_color}\pgfsetfillcolor{eps2pgf_color}
\pgfpathmoveto{\pgfqpoint{0.273cm}{1.395cm}}
\pgfpathcurveto{\pgfqpoint{0.273cm}{1.432cm}}{\pgfqpoint{0.259cm}{1.467cm}}{\pgfqpoint{0.233cm}{1.492cm}}
\pgfpathcurveto{\pgfqpoint{0.207cm}{1.518cm}}{\pgfqpoint{0.173cm}{1.532cm}}{\pgfqpoint{0.137cm}{1.532cm}}
\pgfpathcurveto{\pgfqpoint{0.1cm}{1.532cm}}{\pgfqpoint{0.066cm}{1.518cm}}{\pgfqpoint{0.04cm}{1.492cm}}
\pgfpathcurveto{\pgfqpoint{0.014cm}{1.467cm}}{\pgfqpoint{0cm}{1.432cm}}{\pgfqpoint{0cm}{1.395cm}}
\pgfpathcurveto{\pgfqpoint{0cm}{1.359cm}}{\pgfqpoint{0.014cm}{1.324cm}}{\pgfqpoint{0.04cm}{1.299cm}}
\pgfpathcurveto{\pgfqpoint{0.066cm}{1.273cm}}{\pgfqpoint{0.1cm}{1.258cm}}{\pgfqpoint{0.137cm}{1.258cm}}
\pgfpathcurveto{\pgfqpoint{0.173cm}{1.258cm}}{\pgfqpoint{0.207cm}{1.273cm}}{\pgfqpoint{0.233cm}{1.299cm}}
\pgfpathcurveto{\pgfqpoint{0.259cm}{1.324cm}}{\pgfqpoint{0.273cm}{1.359cm}}{\pgfqpoint{0.273cm}{1.395cm}}
\pgfusepath{fill}
\begin{pgfscope}
\pgfsetdash{}{0cm}
\pgfsetlinewidth{0.818mm}
\pgfsetmiterlimit{7.0}
\pgfpathmoveto{\pgfqpoint{0.682cm}{0.671cm}}
\pgfpathlineto{\pgfqpoint{0.679cm}{1.418cm}}
\pgfusepath{stroke}
\end{pgfscope}
\pgfpathmoveto{\pgfqpoint{0.815cm}{1.399cm}}
\pgfpathcurveto{\pgfqpoint{0.815cm}{1.435cm}}{\pgfqpoint{0.801cm}{1.47cm}}{\pgfqpoint{0.775cm}{1.496cm}}
\pgfpathcurveto{\pgfqpoint{0.75cm}{1.521cm}}{\pgfqpoint{0.715cm}{1.536cm}}{\pgfqpoint{0.679cm}{1.536cm}}
\pgfpathcurveto{\pgfqpoint{0.643cm}{1.536cm}}{\pgfqpoint{0.608cm}{1.521cm}}{\pgfqpoint{0.582cm}{1.496cm}}
\pgfpathcurveto{\pgfqpoint{0.557cm}{1.47cm}}{\pgfqpoint{0.542cm}{1.435cm}}{\pgfqpoint{0.542cm}{1.399cm}}
\pgfpathcurveto{\pgfqpoint{0.542cm}{1.363cm}}{\pgfqpoint{0.557cm}{1.328cm}}{\pgfqpoint{0.582cm}{1.302cm}}
\pgfpathcurveto{\pgfqpoint{0.608cm}{1.276cm}}{\pgfqpoint{0.643cm}{1.262cm}}{\pgfqpoint{0.679cm}{1.262cm}}
\pgfpathcurveto{\pgfqpoint{0.715cm}{1.262cm}}{\pgfqpoint{0.75cm}{1.276cm}}{\pgfqpoint{0.775cm}{1.302cm}}
\pgfpathcurveto{\pgfqpoint{0.801cm}{1.328cm}}{\pgfqpoint{0.815cm}{1.363cm}}{\pgfqpoint{0.815cm}{1.399cm}}
\pgfusepath{fill}
\pgfpathmoveto{\pgfqpoint{1.345cm}{1.371cm}}
\pgfpathcurveto{\pgfqpoint{1.345cm}{1.408cm}}{\pgfqpoint{1.331cm}{1.442cm}}{\pgfqpoint{1.305cm}{1.468cm}}
\pgfpathcurveto{\pgfqpoint{1.28cm}{1.494cm}}{\pgfqpoint{1.245cm}{1.508cm}}{\pgfqpoint{1.209cm}{1.508cm}}
\pgfpathcurveto{\pgfqpoint{1.172cm}{1.508cm}}{\pgfqpoint{1.138cm}{1.494cm}}{\pgfqpoint{1.112cm}{1.468cm}}
\pgfpathcurveto{\pgfqpoint{1.087cm}{1.442cm}}{\pgfqpoint{1.072cm}{1.408cm}}{\pgfqpoint{1.072cm}{1.371cm}}
\pgfpathcurveto{\pgfqpoint{1.072cm}{1.335cm}}{\pgfqpoint{1.087cm}{1.3cm}}{\pgfqpoint{1.112cm}{1.274cm}}
\pgfpathcurveto{\pgfqpoint{1.138cm}{1.249cm}}{\pgfqpoint{1.172cm}{1.234cm}}{\pgfqpoint{1.209cm}{1.234cm}}
\pgfpathcurveto{\pgfqpoint{1.245cm}{1.234cm}}{\pgfqpoint{1.28cm}{1.249cm}}{\pgfqpoint{1.305cm}{1.274cm}}
\pgfpathcurveto{\pgfqpoint{1.331cm}{1.3cm}}{\pgfqpoint{1.345cm}{1.335cm}}{\pgfqpoint{1.345cm}{1.371cm}}
\pgfusepath{fill}
\begin{pgfscope}
\pgfsetdash{}{0cm}
\pgfsetlinewidth{0.818mm}
\pgfsetroundcap
\pgfsetmiterlimit{4.0}
\pgfpathmoveto{\pgfqpoint{0.682cm}{0.671cm}}
\pgfpathlineto{\pgfqpoint{0.682cm}{0.042cm}}
\pgfusepath{stroke}
\end{pgfscope}
\end{pgfscope}
\end{pgfscope}
\end{pgfscope}
\end{tikzpicture}}}+\phi+\psi)+3X(-X^{\!\resizebox{0.6em}{!}{
\begin{tikzpicture}
\pgfpathmoveto{\pgfqpoint{0cm}{-0.035cm}}
\pgfpathlineto{\pgfqpoint{1.376cm}{-0.035cm}}
\pgfpathlineto{\pgfqpoint{1.376cm}{1.552cm}}
\pgfpathlineto{\pgfqpoint{0cm}{1.552cm}}
\pgfpathclose
\pgfusepath{clip}
\begin{pgfscope}
\begin{pgfscope}
\pgfpathmoveto{\pgfqpoint{0cm}{-0.035cm}}
\pgfpathlineto{\pgfqpoint{1.376cm}{-0.035cm}}
\pgfpathlineto{\pgfqpoint{1.376cm}{1.552cm}}
\pgfpathlineto{\pgfqpoint{0cm}{1.552cm}}
\pgfpathclose
\pgfusepath{clip}
\begin{pgfscope}
\begin{pgfscope}
\pgfsetdash{}{0cm}
\pgfsetlinewidth{0.818mm}
\pgfsetroundcap
\pgfsetroundjoin
\pgfsetmiterlimit{7.0}
\definecolor{eps2pgf_color}{gray}{0}\pgfsetstrokecolor{eps2pgf_color}\pgfsetfillcolor{eps2pgf_color}
\pgfpathmoveto{\pgfqpoint{0.117cm}{1.421cm}}
\pgfpathlineto{\pgfqpoint{0.682cm}{0.671cm}}
\pgfpathlineto{\pgfqpoint{1.246cm}{1.421cm}}
\pgfusepath{stroke}
\end{pgfscope}
\definecolor{eps2pgf_color}{gray}{0}\pgfsetstrokecolor{eps2pgf_color}\pgfsetfillcolor{eps2pgf_color}
\pgfpathmoveto{\pgfqpoint{0.273cm}{1.395cm}}
\pgfpathcurveto{\pgfqpoint{0.273cm}{1.432cm}}{\pgfqpoint{0.259cm}{1.467cm}}{\pgfqpoint{0.233cm}{1.492cm}}
\pgfpathcurveto{\pgfqpoint{0.207cm}{1.518cm}}{\pgfqpoint{0.173cm}{1.532cm}}{\pgfqpoint{0.137cm}{1.532cm}}
\pgfpathcurveto{\pgfqpoint{0.1cm}{1.532cm}}{\pgfqpoint{0.066cm}{1.518cm}}{\pgfqpoint{0.04cm}{1.492cm}}
\pgfpathcurveto{\pgfqpoint{0.014cm}{1.467cm}}{\pgfqpoint{0cm}{1.432cm}}{\pgfqpoint{0cm}{1.395cm}}
\pgfpathcurveto{\pgfqpoint{0cm}{1.359cm}}{\pgfqpoint{0.014cm}{1.324cm}}{\pgfqpoint{0.04cm}{1.299cm}}
\pgfpathcurveto{\pgfqpoint{0.066cm}{1.273cm}}{\pgfqpoint{0.1cm}{1.258cm}}{\pgfqpoint{0.137cm}{1.258cm}}
\pgfpathcurveto{\pgfqpoint{0.173cm}{1.258cm}}{\pgfqpoint{0.207cm}{1.273cm}}{\pgfqpoint{0.233cm}{1.299cm}}
\pgfpathcurveto{\pgfqpoint{0.259cm}{1.324cm}}{\pgfqpoint{0.273cm}{1.359cm}}{\pgfqpoint{0.273cm}{1.395cm}}
\pgfusepath{fill}
\begin{pgfscope}
\pgfsetdash{}{0cm}
\pgfsetlinewidth{0.818mm}
\pgfsetmiterlimit{7.0}
\pgfpathmoveto{\pgfqpoint{0.682cm}{0.671cm}}
\pgfpathlineto{\pgfqpoint{0.679cm}{1.418cm}}
\pgfusepath{stroke}
\end{pgfscope}
\pgfpathmoveto{\pgfqpoint{0.815cm}{1.399cm}}
\pgfpathcurveto{\pgfqpoint{0.815cm}{1.435cm}}{\pgfqpoint{0.801cm}{1.47cm}}{\pgfqpoint{0.775cm}{1.496cm}}
\pgfpathcurveto{\pgfqpoint{0.75cm}{1.521cm}}{\pgfqpoint{0.715cm}{1.536cm}}{\pgfqpoint{0.679cm}{1.536cm}}
\pgfpathcurveto{\pgfqpoint{0.643cm}{1.536cm}}{\pgfqpoint{0.608cm}{1.521cm}}{\pgfqpoint{0.582cm}{1.496cm}}
\pgfpathcurveto{\pgfqpoint{0.557cm}{1.47cm}}{\pgfqpoint{0.542cm}{1.435cm}}{\pgfqpoint{0.542cm}{1.399cm}}
\pgfpathcurveto{\pgfqpoint{0.542cm}{1.363cm}}{\pgfqpoint{0.557cm}{1.328cm}}{\pgfqpoint{0.582cm}{1.302cm}}
\pgfpathcurveto{\pgfqpoint{0.608cm}{1.276cm}}{\pgfqpoint{0.643cm}{1.262cm}}{\pgfqpoint{0.679cm}{1.262cm}}
\pgfpathcurveto{\pgfqpoint{0.715cm}{1.262cm}}{\pgfqpoint{0.75cm}{1.276cm}}{\pgfqpoint{0.775cm}{1.302cm}}
\pgfpathcurveto{\pgfqpoint{0.801cm}{1.328cm}}{\pgfqpoint{0.815cm}{1.363cm}}{\pgfqpoint{0.815cm}{1.399cm}}
\pgfusepath{fill}
\pgfpathmoveto{\pgfqpoint{1.345cm}{1.371cm}}
\pgfpathcurveto{\pgfqpoint{1.345cm}{1.408cm}}{\pgfqpoint{1.331cm}{1.442cm}}{\pgfqpoint{1.305cm}{1.468cm}}
\pgfpathcurveto{\pgfqpoint{1.28cm}{1.494cm}}{\pgfqpoint{1.245cm}{1.508cm}}{\pgfqpoint{1.209cm}{1.508cm}}
\pgfpathcurveto{\pgfqpoint{1.172cm}{1.508cm}}{\pgfqpoint{1.138cm}{1.494cm}}{\pgfqpoint{1.112cm}{1.468cm}}
\pgfpathcurveto{\pgfqpoint{1.087cm}{1.442cm}}{\pgfqpoint{1.072cm}{1.408cm}}{\pgfqpoint{1.072cm}{1.371cm}}
\pgfpathcurveto{\pgfqpoint{1.072cm}{1.335cm}}{\pgfqpoint{1.087cm}{1.3cm}}{\pgfqpoint{1.112cm}{1.274cm}}
\pgfpathcurveto{\pgfqpoint{1.138cm}{1.249cm}}{\pgfqpoint{1.172cm}{1.234cm}}{\pgfqpoint{1.209cm}{1.234cm}}
\pgfpathcurveto{\pgfqpoint{1.245cm}{1.234cm}}{\pgfqpoint{1.28cm}{1.249cm}}{\pgfqpoint{1.305cm}{1.274cm}}
\pgfpathcurveto{\pgfqpoint{1.331cm}{1.3cm}}{\pgfqpoint{1.345cm}{1.335cm}}{\pgfqpoint{1.345cm}{1.371cm}}
\pgfusepath{fill}
\begin{pgfscope}
\pgfsetdash{}{0cm}
\pgfsetlinewidth{0.818mm}
\pgfsetroundcap
\pgfsetmiterlimit{4.0}
\pgfpathmoveto{\pgfqpoint{0.682cm}{0.671cm}}
\pgfpathlineto{\pgfqpoint{0.682cm}{0.042cm}}
\pgfusepath{stroke}
\end{pgfscope}
\end{pgfscope}
\end{pgfscope}
\end{pgfscope}
\end{tikzpicture}}}+\phi+\psi)^2+(-X^{\!\resizebox{0.6em}{!}{
\begin{tikzpicture}
\pgfpathmoveto{\pgfqpoint{0cm}{-0.035cm}}
\pgfpathlineto{\pgfqpoint{1.376cm}{-0.035cm}}
\pgfpathlineto{\pgfqpoint{1.376cm}{1.552cm}}
\pgfpathlineto{\pgfqpoint{0cm}{1.552cm}}
\pgfpathclose
\pgfusepath{clip}
\begin{pgfscope}
\begin{pgfscope}
\pgfpathmoveto{\pgfqpoint{0cm}{-0.035cm}}
\pgfpathlineto{\pgfqpoint{1.376cm}{-0.035cm}}
\pgfpathlineto{\pgfqpoint{1.376cm}{1.552cm}}
\pgfpathlineto{\pgfqpoint{0cm}{1.552cm}}
\pgfpathclose
\pgfusepath{clip}
\begin{pgfscope}
\begin{pgfscope}
\pgfsetdash{}{0cm}
\pgfsetlinewidth{0.818mm}
\pgfsetroundcap
\pgfsetroundjoin
\pgfsetmiterlimit{7.0}
\definecolor{eps2pgf_color}{gray}{0}\pgfsetstrokecolor{eps2pgf_color}\pgfsetfillcolor{eps2pgf_color}
\pgfpathmoveto{\pgfqpoint{0.117cm}{1.421cm}}
\pgfpathlineto{\pgfqpoint{0.682cm}{0.671cm}}
\pgfpathlineto{\pgfqpoint{1.246cm}{1.421cm}}
\pgfusepath{stroke}
\end{pgfscope}
\definecolor{eps2pgf_color}{gray}{0}\pgfsetstrokecolor{eps2pgf_color}\pgfsetfillcolor{eps2pgf_color}
\pgfpathmoveto{\pgfqpoint{0.273cm}{1.395cm}}
\pgfpathcurveto{\pgfqpoint{0.273cm}{1.432cm}}{\pgfqpoint{0.259cm}{1.467cm}}{\pgfqpoint{0.233cm}{1.492cm}}
\pgfpathcurveto{\pgfqpoint{0.207cm}{1.518cm}}{\pgfqpoint{0.173cm}{1.532cm}}{\pgfqpoint{0.137cm}{1.532cm}}
\pgfpathcurveto{\pgfqpoint{0.1cm}{1.532cm}}{\pgfqpoint{0.066cm}{1.518cm}}{\pgfqpoint{0.04cm}{1.492cm}}
\pgfpathcurveto{\pgfqpoint{0.014cm}{1.467cm}}{\pgfqpoint{0cm}{1.432cm}}{\pgfqpoint{0cm}{1.395cm}}
\pgfpathcurveto{\pgfqpoint{0cm}{1.359cm}}{\pgfqpoint{0.014cm}{1.324cm}}{\pgfqpoint{0.04cm}{1.299cm}}
\pgfpathcurveto{\pgfqpoint{0.066cm}{1.273cm}}{\pgfqpoint{0.1cm}{1.258cm}}{\pgfqpoint{0.137cm}{1.258cm}}
\pgfpathcurveto{\pgfqpoint{0.173cm}{1.258cm}}{\pgfqpoint{0.207cm}{1.273cm}}{\pgfqpoint{0.233cm}{1.299cm}}
\pgfpathcurveto{\pgfqpoint{0.259cm}{1.324cm}}{\pgfqpoint{0.273cm}{1.359cm}}{\pgfqpoint{0.273cm}{1.395cm}}
\pgfusepath{fill}
\begin{pgfscope}
\pgfsetdash{}{0cm}
\pgfsetlinewidth{0.818mm}
\pgfsetmiterlimit{7.0}
\pgfpathmoveto{\pgfqpoint{0.682cm}{0.671cm}}
\pgfpathlineto{\pgfqpoint{0.679cm}{1.418cm}}
\pgfusepath{stroke}
\end{pgfscope}
\pgfpathmoveto{\pgfqpoint{0.815cm}{1.399cm}}
\pgfpathcurveto{\pgfqpoint{0.815cm}{1.435cm}}{\pgfqpoint{0.801cm}{1.47cm}}{\pgfqpoint{0.775cm}{1.496cm}}
\pgfpathcurveto{\pgfqpoint{0.75cm}{1.521cm}}{\pgfqpoint{0.715cm}{1.536cm}}{\pgfqpoint{0.679cm}{1.536cm}}
\pgfpathcurveto{\pgfqpoint{0.643cm}{1.536cm}}{\pgfqpoint{0.608cm}{1.521cm}}{\pgfqpoint{0.582cm}{1.496cm}}
\pgfpathcurveto{\pgfqpoint{0.557cm}{1.47cm}}{\pgfqpoint{0.542cm}{1.435cm}}{\pgfqpoint{0.542cm}{1.399cm}}
\pgfpathcurveto{\pgfqpoint{0.542cm}{1.363cm}}{\pgfqpoint{0.557cm}{1.328cm}}{\pgfqpoint{0.582cm}{1.302cm}}
\pgfpathcurveto{\pgfqpoint{0.608cm}{1.276cm}}{\pgfqpoint{0.643cm}{1.262cm}}{\pgfqpoint{0.679cm}{1.262cm}}
\pgfpathcurveto{\pgfqpoint{0.715cm}{1.262cm}}{\pgfqpoint{0.75cm}{1.276cm}}{\pgfqpoint{0.775cm}{1.302cm}}
\pgfpathcurveto{\pgfqpoint{0.801cm}{1.328cm}}{\pgfqpoint{0.815cm}{1.363cm}}{\pgfqpoint{0.815cm}{1.399cm}}
\pgfusepath{fill}
\pgfpathmoveto{\pgfqpoint{1.345cm}{1.371cm}}
\pgfpathcurveto{\pgfqpoint{1.345cm}{1.408cm}}{\pgfqpoint{1.331cm}{1.442cm}}{\pgfqpoint{1.305cm}{1.468cm}}
\pgfpathcurveto{\pgfqpoint{1.28cm}{1.494cm}}{\pgfqpoint{1.245cm}{1.508cm}}{\pgfqpoint{1.209cm}{1.508cm}}
\pgfpathcurveto{\pgfqpoint{1.172cm}{1.508cm}}{\pgfqpoint{1.138cm}{1.494cm}}{\pgfqpoint{1.112cm}{1.468cm}}
\pgfpathcurveto{\pgfqpoint{1.087cm}{1.442cm}}{\pgfqpoint{1.072cm}{1.408cm}}{\pgfqpoint{1.072cm}{1.371cm}}
\pgfpathcurveto{\pgfqpoint{1.072cm}{1.335cm}}{\pgfqpoint{1.087cm}{1.3cm}}{\pgfqpoint{1.112cm}{1.274cm}}
\pgfpathcurveto{\pgfqpoint{1.138cm}{1.249cm}}{\pgfqpoint{1.172cm}{1.234cm}}{\pgfqpoint{1.209cm}{1.234cm}}
\pgfpathcurveto{\pgfqpoint{1.245cm}{1.234cm}}{\pgfqpoint{1.28cm}{1.249cm}}{\pgfqpoint{1.305cm}{1.274cm}}
\pgfpathcurveto{\pgfqpoint{1.331cm}{1.3cm}}{\pgfqpoint{1.345cm}{1.335cm}}{\pgfqpoint{1.345cm}{1.371cm}}
\pgfusepath{fill}
\begin{pgfscope}
\pgfsetdash{}{0cm}
\pgfsetlinewidth{0.818mm}
\pgfsetroundcap
\pgfsetmiterlimit{4.0}
\pgfpathmoveto{\pgfqpoint{0.682cm}{0.671cm}}
\pgfpathlineto{\pgfqpoint{0.682cm}{0.042cm}}
\pgfusepath{stroke}
\end{pgfscope}
\end{pgfscope}
\end{pgfscope}
\end{pgfscope}
\end{tikzpicture}}}+\phi+\psi)^3+3b\varphi.
\end{aligned}
\end{align*}

In Section \ref{sec:42} we were able to apply the elliptic a priori estimates pointwise in time.
In view of the decomposition of the elliptic $\Phi^{4}_{5}$ model, this is no longer possible here. The difficulty arises in the paracontrolled ansatz similar to \eqref{eq:th} and the associated commutator as in \eqref{eq:rhs45a}. Indeed, since the differential operator in \eqref{eq:43} includes also time derivative, we require sufficient time regularity in order to control the  commutator. To this end, we make use of the time-mollified paraproduct, introduced in Section \ref{ssec:para}, and  apply Lemma \ref{lem:5.1}. More precisely, we assume that $\phi$ is paracontrolled by $X^{\!\resizebox{0.6em}{!}{
\begin{tikzpicture}
\pgfpathmoveto{\pgfqpoint{0cm}{0cm}}
\pgfpathlineto{\pgfqpoint{1.376cm}{0cm}}
\pgfpathlineto{\pgfqpoint{1.376cm}{1.588cm}}
\pgfpathlineto{\pgfqpoint{0cm}{1.588cm}}
\pgfpathclose
\pgfusepath{clip}
\begin{pgfscope}
\begin{pgfscope}
\pgfpathmoveto{\pgfqpoint{0cm}{0cm}}
\pgfpathlineto{\pgfqpoint{1.376cm}{0cm}}
\pgfpathlineto{\pgfqpoint{1.376cm}{1.588cm}}
\pgfpathlineto{\pgfqpoint{0cm}{1.588cm}}
\pgfpathclose
\pgfusepath{clip}
\begin{pgfscope}
\begin{pgfscope}
\definecolor{eps2pgf_color}{gray}{0.976471}\pgfsetstrokecolor{eps2pgf_color}\pgfsetfillcolor{eps2pgf_color}
\pgfpathmoveto{\pgfqpoint{0cm}{0cm}}
\pgfpathlineto{\pgfqpoint{1.376cm}{0cm}}
\pgfpathlineto{\pgfqpoint{1.376cm}{1.588cm}}
\pgfpathlineto{\pgfqpoint{0cm}{1.588cm}}
\pgfpathclose
\pgfusepath{fill}
\end{pgfscope}
\begin{pgfscope}
\pgfsetdash{}{0cm}
\pgfsetlinewidth{0.818mm}
\pgfsetroundcap
\pgfsetroundjoin
\pgfsetmiterlimit{7.0}
\definecolor{eps2pgf_color}{gray}{0}\pgfsetstrokecolor{eps2pgf_color}\pgfsetfillcolor{eps2pgf_color}
\pgfpathmoveto{\pgfqpoint{0.117cm}{1.476cm}}
\pgfpathlineto{\pgfqpoint{0.682cm}{0.726cm}}
\pgfpathlineto{\pgfqpoint{1.246cm}{1.476cm}}
\pgfusepath{stroke}
\end{pgfscope}
\definecolor{eps2pgf_color}{gray}{0}\pgfsetstrokecolor{eps2pgf_color}\pgfsetfillcolor{eps2pgf_color}
\pgfpathmoveto{\pgfqpoint{0.273cm}{1.451cm}}
\pgfpathcurveto{\pgfqpoint{0.273cm}{1.487cm}}{\pgfqpoint{0.259cm}{1.522cm}}{\pgfqpoint{0.233cm}{1.547cm}}
\pgfpathcurveto{\pgfqpoint{0.207cm}{1.573cm}}{\pgfqpoint{0.173cm}{1.588cm}}{\pgfqpoint{0.137cm}{1.588cm}}
\pgfpathcurveto{\pgfqpoint{0.1cm}{1.588cm}}{\pgfqpoint{0.066cm}{1.573cm}}{\pgfqpoint{0.04cm}{1.547cm}}
\pgfpathcurveto{\pgfqpoint{0.014cm}{1.522cm}}{\pgfqpoint{0cm}{1.487cm}}{\pgfqpoint{0cm}{1.451cm}}
\pgfpathcurveto{\pgfqpoint{0cm}{1.414cm}}{\pgfqpoint{0.014cm}{1.379cm}}{\pgfqpoint{0.04cm}{1.354cm}}
\pgfpathcurveto{\pgfqpoint{0.066cm}{1.328cm}}{\pgfqpoint{0.1cm}{1.314cm}}{\pgfqpoint{0.137cm}{1.314cm}}
\pgfpathcurveto{\pgfqpoint{0.173cm}{1.314cm}}{\pgfqpoint{0.207cm}{1.328cm}}{\pgfqpoint{0.233cm}{1.354cm}}
\pgfpathcurveto{\pgfqpoint{0.259cm}{1.379cm}}{\pgfqpoint{0.273cm}{1.414cm}}{\pgfqpoint{0.273cm}{1.451cm}}
\pgfusepath{fill}
\pgfpathmoveto{\pgfqpoint{1.345cm}{1.426cm}}
\pgfpathcurveto{\pgfqpoint{1.345cm}{1.463cm}}{\pgfqpoint{1.331cm}{1.497cm}}{\pgfqpoint{1.305cm}{1.523cm}}
\pgfpathcurveto{\pgfqpoint{1.28cm}{1.549cm}}{\pgfqpoint{1.245cm}{1.563cm}}{\pgfqpoint{1.209cm}{1.563cm}}
\pgfpathcurveto{\pgfqpoint{1.172cm}{1.563cm}}{\pgfqpoint{1.138cm}{1.549cm}}{\pgfqpoint{1.112cm}{1.523cm}}
\pgfpathcurveto{\pgfqpoint{1.087cm}{1.497cm}}{\pgfqpoint{1.072cm}{1.463cm}}{\pgfqpoint{1.072cm}{1.426cm}}
\pgfpathcurveto{\pgfqpoint{1.072cm}{1.39cm}}{\pgfqpoint{1.087cm}{1.355cm}}{\pgfqpoint{1.112cm}{1.329cm}}
\pgfpathcurveto{\pgfqpoint{1.138cm}{1.304cm}}{\pgfqpoint{1.172cm}{1.289cm}}{\pgfqpoint{1.209cm}{1.289cm}}
\pgfpathcurveto{\pgfqpoint{1.245cm}{1.289cm}}{\pgfqpoint{1.28cm}{1.304cm}}{\pgfqpoint{1.305cm}{1.329cm}}
\pgfpathcurveto{\pgfqpoint{1.331cm}{1.355cm}}{\pgfqpoint{1.345cm}{1.39cm}}{\pgfqpoint{1.345cm}{1.426cm}}
\pgfusepath{fill}
\begin{pgfscope}
\pgfsetdash{}{0cm}
\pgfsetlinewidth{0.818mm}
\pgfsetroundcap
\pgfsetmiterlimit{4.0}
\pgfpathmoveto{\pgfqpoint{0.682cm}{0.726cm}}
\pgfpathlineto{\pgfqpoint{0.682cm}{0.097cm}}
\pgfusepath{stroke}
\end{pgfscope}
\end{pgfscope}
\end{pgfscope}
\end{pgfscope}
\end{tikzpicture}}}$, namely, it holds
\begin{align}\label{eq:th43}
\phi=\vartheta -3(-X^{\!\resizebox{0.6em}{!}{
\begin{tikzpicture}
\pgfpathmoveto{\pgfqpoint{0cm}{-0.035cm}}
\pgfpathlineto{\pgfqpoint{1.376cm}{-0.035cm}}
\pgfpathlineto{\pgfqpoint{1.376cm}{1.552cm}}
\pgfpathlineto{\pgfqpoint{0cm}{1.552cm}}
\pgfpathclose
\pgfusepath{clip}
\begin{pgfscope}
\begin{pgfscope}
\pgfpathmoveto{\pgfqpoint{0cm}{-0.035cm}}
\pgfpathlineto{\pgfqpoint{1.376cm}{-0.035cm}}
\pgfpathlineto{\pgfqpoint{1.376cm}{1.552cm}}
\pgfpathlineto{\pgfqpoint{0cm}{1.552cm}}
\pgfpathclose
\pgfusepath{clip}
\begin{pgfscope}
\begin{pgfscope}
\pgfsetdash{}{0cm}
\pgfsetlinewidth{0.818mm}
\pgfsetroundcap
\pgfsetroundjoin
\pgfsetmiterlimit{7.0}
\definecolor{eps2pgf_color}{gray}{0}\pgfsetstrokecolor{eps2pgf_color}\pgfsetfillcolor{eps2pgf_color}
\pgfpathmoveto{\pgfqpoint{0.117cm}{1.421cm}}
\pgfpathlineto{\pgfqpoint{0.682cm}{0.671cm}}
\pgfpathlineto{\pgfqpoint{1.246cm}{1.421cm}}
\pgfusepath{stroke}
\end{pgfscope}
\definecolor{eps2pgf_color}{gray}{0}\pgfsetstrokecolor{eps2pgf_color}\pgfsetfillcolor{eps2pgf_color}
\pgfpathmoveto{\pgfqpoint{0.273cm}{1.395cm}}
\pgfpathcurveto{\pgfqpoint{0.273cm}{1.432cm}}{\pgfqpoint{0.259cm}{1.467cm}}{\pgfqpoint{0.233cm}{1.492cm}}
\pgfpathcurveto{\pgfqpoint{0.207cm}{1.518cm}}{\pgfqpoint{0.173cm}{1.532cm}}{\pgfqpoint{0.137cm}{1.532cm}}
\pgfpathcurveto{\pgfqpoint{0.1cm}{1.532cm}}{\pgfqpoint{0.066cm}{1.518cm}}{\pgfqpoint{0.04cm}{1.492cm}}
\pgfpathcurveto{\pgfqpoint{0.014cm}{1.467cm}}{\pgfqpoint{0cm}{1.432cm}}{\pgfqpoint{0cm}{1.395cm}}
\pgfpathcurveto{\pgfqpoint{0cm}{1.359cm}}{\pgfqpoint{0.014cm}{1.324cm}}{\pgfqpoint{0.04cm}{1.299cm}}
\pgfpathcurveto{\pgfqpoint{0.066cm}{1.273cm}}{\pgfqpoint{0.1cm}{1.258cm}}{\pgfqpoint{0.137cm}{1.258cm}}
\pgfpathcurveto{\pgfqpoint{0.173cm}{1.258cm}}{\pgfqpoint{0.207cm}{1.273cm}}{\pgfqpoint{0.233cm}{1.299cm}}
\pgfpathcurveto{\pgfqpoint{0.259cm}{1.324cm}}{\pgfqpoint{0.273cm}{1.359cm}}{\pgfqpoint{0.273cm}{1.395cm}}
\pgfusepath{fill}
\begin{pgfscope}
\pgfsetdash{}{0cm}
\pgfsetlinewidth{0.818mm}
\pgfsetmiterlimit{7.0}
\pgfpathmoveto{\pgfqpoint{0.682cm}{0.671cm}}
\pgfpathlineto{\pgfqpoint{0.679cm}{1.418cm}}
\pgfusepath{stroke}
\end{pgfscope}
\pgfpathmoveto{\pgfqpoint{0.815cm}{1.399cm}}
\pgfpathcurveto{\pgfqpoint{0.815cm}{1.435cm}}{\pgfqpoint{0.801cm}{1.47cm}}{\pgfqpoint{0.775cm}{1.496cm}}
\pgfpathcurveto{\pgfqpoint{0.75cm}{1.521cm}}{\pgfqpoint{0.715cm}{1.536cm}}{\pgfqpoint{0.679cm}{1.536cm}}
\pgfpathcurveto{\pgfqpoint{0.643cm}{1.536cm}}{\pgfqpoint{0.608cm}{1.521cm}}{\pgfqpoint{0.582cm}{1.496cm}}
\pgfpathcurveto{\pgfqpoint{0.557cm}{1.47cm}}{\pgfqpoint{0.542cm}{1.435cm}}{\pgfqpoint{0.542cm}{1.399cm}}
\pgfpathcurveto{\pgfqpoint{0.542cm}{1.363cm}}{\pgfqpoint{0.557cm}{1.328cm}}{\pgfqpoint{0.582cm}{1.302cm}}
\pgfpathcurveto{\pgfqpoint{0.608cm}{1.276cm}}{\pgfqpoint{0.643cm}{1.262cm}}{\pgfqpoint{0.679cm}{1.262cm}}
\pgfpathcurveto{\pgfqpoint{0.715cm}{1.262cm}}{\pgfqpoint{0.75cm}{1.276cm}}{\pgfqpoint{0.775cm}{1.302cm}}
\pgfpathcurveto{\pgfqpoint{0.801cm}{1.328cm}}{\pgfqpoint{0.815cm}{1.363cm}}{\pgfqpoint{0.815cm}{1.399cm}}
\pgfusepath{fill}
\pgfpathmoveto{\pgfqpoint{1.345cm}{1.371cm}}
\pgfpathcurveto{\pgfqpoint{1.345cm}{1.408cm}}{\pgfqpoint{1.331cm}{1.442cm}}{\pgfqpoint{1.305cm}{1.468cm}}
\pgfpathcurveto{\pgfqpoint{1.28cm}{1.494cm}}{\pgfqpoint{1.245cm}{1.508cm}}{\pgfqpoint{1.209cm}{1.508cm}}
\pgfpathcurveto{\pgfqpoint{1.172cm}{1.508cm}}{\pgfqpoint{1.138cm}{1.494cm}}{\pgfqpoint{1.112cm}{1.468cm}}
\pgfpathcurveto{\pgfqpoint{1.087cm}{1.442cm}}{\pgfqpoint{1.072cm}{1.408cm}}{\pgfqpoint{1.072cm}{1.371cm}}
\pgfpathcurveto{\pgfqpoint{1.072cm}{1.335cm}}{\pgfqpoint{1.087cm}{1.3cm}}{\pgfqpoint{1.112cm}{1.274cm}}
\pgfpathcurveto{\pgfqpoint{1.138cm}{1.249cm}}{\pgfqpoint{1.172cm}{1.234cm}}{\pgfqpoint{1.209cm}{1.234cm}}
\pgfpathcurveto{\pgfqpoint{1.245cm}{1.234cm}}{\pgfqpoint{1.28cm}{1.249cm}}{\pgfqpoint{1.305cm}{1.274cm}}
\pgfpathcurveto{\pgfqpoint{1.331cm}{1.3cm}}{\pgfqpoint{1.345cm}{1.335cm}}{\pgfqpoint{1.345cm}{1.371cm}}
\pgfusepath{fill}
\begin{pgfscope}
\pgfsetdash{}{0cm}
\pgfsetlinewidth{0.818mm}
\pgfsetroundcap
\pgfsetmiterlimit{4.0}
\pgfpathmoveto{\pgfqpoint{0.682cm}{0.671cm}}
\pgfpathlineto{\pgfqpoint{0.682cm}{0.042cm}}
\pgfusepath{stroke}
\end{pgfscope}
\end{pgfscope}
\end{pgfscope}
\end{pgfscope}
\end{tikzpicture}}} + \phi + \psi)\Prec X^{\!\resizebox{0.6em}{!}{
\begin{tikzpicture}
\pgfpathmoveto{\pgfqpoint{0cm}{0cm}}
\pgfpathlineto{\pgfqpoint{1.376cm}{0cm}}
\pgfpathlineto{\pgfqpoint{1.376cm}{1.588cm}}
\pgfpathlineto{\pgfqpoint{0cm}{1.588cm}}
\pgfpathclose
\pgfusepath{clip}
\begin{pgfscope}
\begin{pgfscope}
\pgfpathmoveto{\pgfqpoint{0cm}{0cm}}
\pgfpathlineto{\pgfqpoint{1.376cm}{0cm}}
\pgfpathlineto{\pgfqpoint{1.376cm}{1.588cm}}
\pgfpathlineto{\pgfqpoint{0cm}{1.588cm}}
\pgfpathclose
\pgfusepath{clip}
\begin{pgfscope}
\begin{pgfscope}
\definecolor{eps2pgf_color}{gray}{0.976471}\pgfsetstrokecolor{eps2pgf_color}\pgfsetfillcolor{eps2pgf_color}
\pgfpathmoveto{\pgfqpoint{0cm}{0cm}}
\pgfpathlineto{\pgfqpoint{1.376cm}{0cm}}
\pgfpathlineto{\pgfqpoint{1.376cm}{1.588cm}}
\pgfpathlineto{\pgfqpoint{0cm}{1.588cm}}
\pgfpathclose
\pgfusepath{fill}
\end{pgfscope}
\begin{pgfscope}
\pgfsetdash{}{0cm}
\pgfsetlinewidth{0.818mm}
\pgfsetroundcap
\pgfsetroundjoin
\pgfsetmiterlimit{7.0}
\definecolor{eps2pgf_color}{gray}{0}\pgfsetstrokecolor{eps2pgf_color}\pgfsetfillcolor{eps2pgf_color}
\pgfpathmoveto{\pgfqpoint{0.117cm}{1.476cm}}
\pgfpathlineto{\pgfqpoint{0.682cm}{0.726cm}}
\pgfpathlineto{\pgfqpoint{1.246cm}{1.476cm}}
\pgfusepath{stroke}
\end{pgfscope}
\definecolor{eps2pgf_color}{gray}{0}\pgfsetstrokecolor{eps2pgf_color}\pgfsetfillcolor{eps2pgf_color}
\pgfpathmoveto{\pgfqpoint{0.273cm}{1.451cm}}
\pgfpathcurveto{\pgfqpoint{0.273cm}{1.487cm}}{\pgfqpoint{0.259cm}{1.522cm}}{\pgfqpoint{0.233cm}{1.547cm}}
\pgfpathcurveto{\pgfqpoint{0.207cm}{1.573cm}}{\pgfqpoint{0.173cm}{1.588cm}}{\pgfqpoint{0.137cm}{1.588cm}}
\pgfpathcurveto{\pgfqpoint{0.1cm}{1.588cm}}{\pgfqpoint{0.066cm}{1.573cm}}{\pgfqpoint{0.04cm}{1.547cm}}
\pgfpathcurveto{\pgfqpoint{0.014cm}{1.522cm}}{\pgfqpoint{0cm}{1.487cm}}{\pgfqpoint{0cm}{1.451cm}}
\pgfpathcurveto{\pgfqpoint{0cm}{1.414cm}}{\pgfqpoint{0.014cm}{1.379cm}}{\pgfqpoint{0.04cm}{1.354cm}}
\pgfpathcurveto{\pgfqpoint{0.066cm}{1.328cm}}{\pgfqpoint{0.1cm}{1.314cm}}{\pgfqpoint{0.137cm}{1.314cm}}
\pgfpathcurveto{\pgfqpoint{0.173cm}{1.314cm}}{\pgfqpoint{0.207cm}{1.328cm}}{\pgfqpoint{0.233cm}{1.354cm}}
\pgfpathcurveto{\pgfqpoint{0.259cm}{1.379cm}}{\pgfqpoint{0.273cm}{1.414cm}}{\pgfqpoint{0.273cm}{1.451cm}}
\pgfusepath{fill}
\pgfpathmoveto{\pgfqpoint{1.345cm}{1.426cm}}
\pgfpathcurveto{\pgfqpoint{1.345cm}{1.463cm}}{\pgfqpoint{1.331cm}{1.497cm}}{\pgfqpoint{1.305cm}{1.523cm}}
\pgfpathcurveto{\pgfqpoint{1.28cm}{1.549cm}}{\pgfqpoint{1.245cm}{1.563cm}}{\pgfqpoint{1.209cm}{1.563cm}}
\pgfpathcurveto{\pgfqpoint{1.172cm}{1.563cm}}{\pgfqpoint{1.138cm}{1.549cm}}{\pgfqpoint{1.112cm}{1.523cm}}
\pgfpathcurveto{\pgfqpoint{1.087cm}{1.497cm}}{\pgfqpoint{1.072cm}{1.463cm}}{\pgfqpoint{1.072cm}{1.426cm}}
\pgfpathcurveto{\pgfqpoint{1.072cm}{1.39cm}}{\pgfqpoint{1.087cm}{1.355cm}}{\pgfqpoint{1.112cm}{1.329cm}}
\pgfpathcurveto{\pgfqpoint{1.138cm}{1.304cm}}{\pgfqpoint{1.172cm}{1.289cm}}{\pgfqpoint{1.209cm}{1.289cm}}
\pgfpathcurveto{\pgfqpoint{1.245cm}{1.289cm}}{\pgfqpoint{1.28cm}{1.304cm}}{\pgfqpoint{1.305cm}{1.329cm}}
\pgfpathcurveto{\pgfqpoint{1.331cm}{1.355cm}}{\pgfqpoint{1.345cm}{1.39cm}}{\pgfqpoint{1.345cm}{1.426cm}}
\pgfusepath{fill}
\begin{pgfscope}
\pgfsetdash{}{0cm}
\pgfsetlinewidth{0.818mm}
\pgfsetroundcap
\pgfsetmiterlimit{4.0}
\pgfpathmoveto{\pgfqpoint{0.682cm}{0.726cm}}
\pgfpathlineto{\pgfqpoint{0.682cm}{0.097cm}}
\pgfusepath{stroke}
\end{pgfscope}
\end{pgfscope}
\end{pgfscope}
\end{pgfscope}
\end{tikzpicture}}}
\end{align}
for some $\vartheta$ which is more regular (we will see below that $\vartheta$ has the regularity $1+\alpha$).

Furthermore, it can be seen in Lemma \ref{lemma:schauder-par} that the expected time regularity for $\psi$ is not optimal. Indeed,  we cannot go beyond $C^{1}$ with respect to time, which would be natural for taking the full advantage of the interpolation (in time) from Lemma \ref{lemma:interp2} and mimicking the strategy of Section \ref{sec:45}. Therefore, the terms requiring time regularity of $\phi+\psi$ have to be treated differently. To this end, it is necessary to consider a higher power of the weight $\rho$ in the bounds for $\vartheta$, which will help us compensate for sub-optimal time interpolation. Therefore, we aim at estimating $\vartheta$ in $C\CC^{1+\alpha}(\rho^{3+\gamma'})$ where $0<\gamma'<\gamma$. This issue  will become even more challenging in the coming down from infinity in Section \ref{s:d}, where no time interpolation is available.

With \eqref{eq:th43} at hand, \eqref{eq:43} rewrites as
\begin{align}\label{eq:rhs43a}
\begin{aligned}
0&= \LL\vartheta+\LL\psi -3\Big( (-X^{\!\resizebox{0.6em}{!}{
\begin{tikzpicture}
\pgfpathmoveto{\pgfqpoint{0cm}{-0.035cm}}
\pgfpathlineto{\pgfqpoint{1.376cm}{-0.035cm}}
\pgfpathlineto{\pgfqpoint{1.376cm}{1.552cm}}
\pgfpathlineto{\pgfqpoint{0cm}{1.552cm}}
\pgfpathclose
\pgfusepath{clip}
\begin{pgfscope}
\begin{pgfscope}
\pgfpathmoveto{\pgfqpoint{0cm}{-0.035cm}}
\pgfpathlineto{\pgfqpoint{1.376cm}{-0.035cm}}
\pgfpathlineto{\pgfqpoint{1.376cm}{1.552cm}}
\pgfpathlineto{\pgfqpoint{0cm}{1.552cm}}
\pgfpathclose
\pgfusepath{clip}
\begin{pgfscope}
\begin{pgfscope}
\pgfsetdash{}{0cm}
\pgfsetlinewidth{0.818mm}
\pgfsetroundcap
\pgfsetroundjoin
\pgfsetmiterlimit{7.0}
\definecolor{eps2pgf_color}{gray}{0}\pgfsetstrokecolor{eps2pgf_color}\pgfsetfillcolor{eps2pgf_color}
\pgfpathmoveto{\pgfqpoint{0.117cm}{1.421cm}}
\pgfpathlineto{\pgfqpoint{0.682cm}{0.671cm}}
\pgfpathlineto{\pgfqpoint{1.246cm}{1.421cm}}
\pgfusepath{stroke}
\end{pgfscope}
\definecolor{eps2pgf_color}{gray}{0}\pgfsetstrokecolor{eps2pgf_color}\pgfsetfillcolor{eps2pgf_color}
\pgfpathmoveto{\pgfqpoint{0.273cm}{1.395cm}}
\pgfpathcurveto{\pgfqpoint{0.273cm}{1.432cm}}{\pgfqpoint{0.259cm}{1.467cm}}{\pgfqpoint{0.233cm}{1.492cm}}
\pgfpathcurveto{\pgfqpoint{0.207cm}{1.518cm}}{\pgfqpoint{0.173cm}{1.532cm}}{\pgfqpoint{0.137cm}{1.532cm}}
\pgfpathcurveto{\pgfqpoint{0.1cm}{1.532cm}}{\pgfqpoint{0.066cm}{1.518cm}}{\pgfqpoint{0.04cm}{1.492cm}}
\pgfpathcurveto{\pgfqpoint{0.014cm}{1.467cm}}{\pgfqpoint{0cm}{1.432cm}}{\pgfqpoint{0cm}{1.395cm}}
\pgfpathcurveto{\pgfqpoint{0cm}{1.359cm}}{\pgfqpoint{0.014cm}{1.324cm}}{\pgfqpoint{0.04cm}{1.299cm}}
\pgfpathcurveto{\pgfqpoint{0.066cm}{1.273cm}}{\pgfqpoint{0.1cm}{1.258cm}}{\pgfqpoint{0.137cm}{1.258cm}}
\pgfpathcurveto{\pgfqpoint{0.173cm}{1.258cm}}{\pgfqpoint{0.207cm}{1.273cm}}{\pgfqpoint{0.233cm}{1.299cm}}
\pgfpathcurveto{\pgfqpoint{0.259cm}{1.324cm}}{\pgfqpoint{0.273cm}{1.359cm}}{\pgfqpoint{0.273cm}{1.395cm}}
\pgfusepath{fill}
\begin{pgfscope}
\pgfsetdash{}{0cm}
\pgfsetlinewidth{0.818mm}
\pgfsetmiterlimit{7.0}
\pgfpathmoveto{\pgfqpoint{0.682cm}{0.671cm}}
\pgfpathlineto{\pgfqpoint{0.679cm}{1.418cm}}
\pgfusepath{stroke}
\end{pgfscope}
\pgfpathmoveto{\pgfqpoint{0.815cm}{1.399cm}}
\pgfpathcurveto{\pgfqpoint{0.815cm}{1.435cm}}{\pgfqpoint{0.801cm}{1.47cm}}{\pgfqpoint{0.775cm}{1.496cm}}
\pgfpathcurveto{\pgfqpoint{0.75cm}{1.521cm}}{\pgfqpoint{0.715cm}{1.536cm}}{\pgfqpoint{0.679cm}{1.536cm}}
\pgfpathcurveto{\pgfqpoint{0.643cm}{1.536cm}}{\pgfqpoint{0.608cm}{1.521cm}}{\pgfqpoint{0.582cm}{1.496cm}}
\pgfpathcurveto{\pgfqpoint{0.557cm}{1.47cm}}{\pgfqpoint{0.542cm}{1.435cm}}{\pgfqpoint{0.542cm}{1.399cm}}
\pgfpathcurveto{\pgfqpoint{0.542cm}{1.363cm}}{\pgfqpoint{0.557cm}{1.328cm}}{\pgfqpoint{0.582cm}{1.302cm}}
\pgfpathcurveto{\pgfqpoint{0.608cm}{1.276cm}}{\pgfqpoint{0.643cm}{1.262cm}}{\pgfqpoint{0.679cm}{1.262cm}}
\pgfpathcurveto{\pgfqpoint{0.715cm}{1.262cm}}{\pgfqpoint{0.75cm}{1.276cm}}{\pgfqpoint{0.775cm}{1.302cm}}
\pgfpathcurveto{\pgfqpoint{0.801cm}{1.328cm}}{\pgfqpoint{0.815cm}{1.363cm}}{\pgfqpoint{0.815cm}{1.399cm}}
\pgfusepath{fill}
\pgfpathmoveto{\pgfqpoint{1.345cm}{1.371cm}}
\pgfpathcurveto{\pgfqpoint{1.345cm}{1.408cm}}{\pgfqpoint{1.331cm}{1.442cm}}{\pgfqpoint{1.305cm}{1.468cm}}
\pgfpathcurveto{\pgfqpoint{1.28cm}{1.494cm}}{\pgfqpoint{1.245cm}{1.508cm}}{\pgfqpoint{1.209cm}{1.508cm}}
\pgfpathcurveto{\pgfqpoint{1.172cm}{1.508cm}}{\pgfqpoint{1.138cm}{1.494cm}}{\pgfqpoint{1.112cm}{1.468cm}}
\pgfpathcurveto{\pgfqpoint{1.087cm}{1.442cm}}{\pgfqpoint{1.072cm}{1.408cm}}{\pgfqpoint{1.072cm}{1.371cm}}
\pgfpathcurveto{\pgfqpoint{1.072cm}{1.335cm}}{\pgfqpoint{1.087cm}{1.3cm}}{\pgfqpoint{1.112cm}{1.274cm}}
\pgfpathcurveto{\pgfqpoint{1.138cm}{1.249cm}}{\pgfqpoint{1.172cm}{1.234cm}}{\pgfqpoint{1.209cm}{1.234cm}}
\pgfpathcurveto{\pgfqpoint{1.245cm}{1.234cm}}{\pgfqpoint{1.28cm}{1.249cm}}{\pgfqpoint{1.305cm}{1.274cm}}
\pgfpathcurveto{\pgfqpoint{1.331cm}{1.3cm}}{\pgfqpoint{1.345cm}{1.335cm}}{\pgfqpoint{1.345cm}{1.371cm}}
\pgfusepath{fill}
\begin{pgfscope}
\pgfsetdash{}{0cm}
\pgfsetlinewidth{0.818mm}
\pgfsetroundcap
\pgfsetmiterlimit{4.0}
\pgfpathmoveto{\pgfqpoint{0.682cm}{0.671cm}}
\pgfpathlineto{\pgfqpoint{0.682cm}{0.042cm}}
\pgfusepath{stroke}
\end{pgfscope}
\end{pgfscope}
\end{pgfscope}
\end{pgfscope}
\end{tikzpicture}}}+\phi+\psi)\Prec\llbracket X^{2}\rrbracket- (-X^{\!\resizebox{0.6em}{!}{
\begin{tikzpicture}
\pgfpathmoveto{\pgfqpoint{0cm}{-0.035cm}}
\pgfpathlineto{\pgfqpoint{1.376cm}{-0.035cm}}
\pgfpathlineto{\pgfqpoint{1.376cm}{1.552cm}}
\pgfpathlineto{\pgfqpoint{0cm}{1.552cm}}
\pgfpathclose
\pgfusepath{clip}
\begin{pgfscope}
\begin{pgfscope}
\pgfpathmoveto{\pgfqpoint{0cm}{-0.035cm}}
\pgfpathlineto{\pgfqpoint{1.376cm}{-0.035cm}}
\pgfpathlineto{\pgfqpoint{1.376cm}{1.552cm}}
\pgfpathlineto{\pgfqpoint{0cm}{1.552cm}}
\pgfpathclose
\pgfusepath{clip}
\begin{pgfscope}
\begin{pgfscope}
\pgfsetdash{}{0cm}
\pgfsetlinewidth{0.818mm}
\pgfsetroundcap
\pgfsetroundjoin
\pgfsetmiterlimit{7.0}
\definecolor{eps2pgf_color}{gray}{0}\pgfsetstrokecolor{eps2pgf_color}\pgfsetfillcolor{eps2pgf_color}
\pgfpathmoveto{\pgfqpoint{0.117cm}{1.421cm}}
\pgfpathlineto{\pgfqpoint{0.682cm}{0.671cm}}
\pgfpathlineto{\pgfqpoint{1.246cm}{1.421cm}}
\pgfusepath{stroke}
\end{pgfscope}
\definecolor{eps2pgf_color}{gray}{0}\pgfsetstrokecolor{eps2pgf_color}\pgfsetfillcolor{eps2pgf_color}
\pgfpathmoveto{\pgfqpoint{0.273cm}{1.395cm}}
\pgfpathcurveto{\pgfqpoint{0.273cm}{1.432cm}}{\pgfqpoint{0.259cm}{1.467cm}}{\pgfqpoint{0.233cm}{1.492cm}}
\pgfpathcurveto{\pgfqpoint{0.207cm}{1.518cm}}{\pgfqpoint{0.173cm}{1.532cm}}{\pgfqpoint{0.137cm}{1.532cm}}
\pgfpathcurveto{\pgfqpoint{0.1cm}{1.532cm}}{\pgfqpoint{0.066cm}{1.518cm}}{\pgfqpoint{0.04cm}{1.492cm}}
\pgfpathcurveto{\pgfqpoint{0.014cm}{1.467cm}}{\pgfqpoint{0cm}{1.432cm}}{\pgfqpoint{0cm}{1.395cm}}
\pgfpathcurveto{\pgfqpoint{0cm}{1.359cm}}{\pgfqpoint{0.014cm}{1.324cm}}{\pgfqpoint{0.04cm}{1.299cm}}
\pgfpathcurveto{\pgfqpoint{0.066cm}{1.273cm}}{\pgfqpoint{0.1cm}{1.258cm}}{\pgfqpoint{0.137cm}{1.258cm}}
\pgfpathcurveto{\pgfqpoint{0.173cm}{1.258cm}}{\pgfqpoint{0.207cm}{1.273cm}}{\pgfqpoint{0.233cm}{1.299cm}}
\pgfpathcurveto{\pgfqpoint{0.259cm}{1.324cm}}{\pgfqpoint{0.273cm}{1.359cm}}{\pgfqpoint{0.273cm}{1.395cm}}
\pgfusepath{fill}
\begin{pgfscope}
\pgfsetdash{}{0cm}
\pgfsetlinewidth{0.818mm}
\pgfsetmiterlimit{7.0}
\pgfpathmoveto{\pgfqpoint{0.682cm}{0.671cm}}
\pgfpathlineto{\pgfqpoint{0.679cm}{1.418cm}}
\pgfusepath{stroke}
\end{pgfscope}
\pgfpathmoveto{\pgfqpoint{0.815cm}{1.399cm}}
\pgfpathcurveto{\pgfqpoint{0.815cm}{1.435cm}}{\pgfqpoint{0.801cm}{1.47cm}}{\pgfqpoint{0.775cm}{1.496cm}}
\pgfpathcurveto{\pgfqpoint{0.75cm}{1.521cm}}{\pgfqpoint{0.715cm}{1.536cm}}{\pgfqpoint{0.679cm}{1.536cm}}
\pgfpathcurveto{\pgfqpoint{0.643cm}{1.536cm}}{\pgfqpoint{0.608cm}{1.521cm}}{\pgfqpoint{0.582cm}{1.496cm}}
\pgfpathcurveto{\pgfqpoint{0.557cm}{1.47cm}}{\pgfqpoint{0.542cm}{1.435cm}}{\pgfqpoint{0.542cm}{1.399cm}}
\pgfpathcurveto{\pgfqpoint{0.542cm}{1.363cm}}{\pgfqpoint{0.557cm}{1.328cm}}{\pgfqpoint{0.582cm}{1.302cm}}
\pgfpathcurveto{\pgfqpoint{0.608cm}{1.276cm}}{\pgfqpoint{0.643cm}{1.262cm}}{\pgfqpoint{0.679cm}{1.262cm}}
\pgfpathcurveto{\pgfqpoint{0.715cm}{1.262cm}}{\pgfqpoint{0.75cm}{1.276cm}}{\pgfqpoint{0.775cm}{1.302cm}}
\pgfpathcurveto{\pgfqpoint{0.801cm}{1.328cm}}{\pgfqpoint{0.815cm}{1.363cm}}{\pgfqpoint{0.815cm}{1.399cm}}
\pgfusepath{fill}
\pgfpathmoveto{\pgfqpoint{1.345cm}{1.371cm}}
\pgfpathcurveto{\pgfqpoint{1.345cm}{1.408cm}}{\pgfqpoint{1.331cm}{1.442cm}}{\pgfqpoint{1.305cm}{1.468cm}}
\pgfpathcurveto{\pgfqpoint{1.28cm}{1.494cm}}{\pgfqpoint{1.245cm}{1.508cm}}{\pgfqpoint{1.209cm}{1.508cm}}
\pgfpathcurveto{\pgfqpoint{1.172cm}{1.508cm}}{\pgfqpoint{1.138cm}{1.494cm}}{\pgfqpoint{1.112cm}{1.468cm}}
\pgfpathcurveto{\pgfqpoint{1.087cm}{1.442cm}}{\pgfqpoint{1.072cm}{1.408cm}}{\pgfqpoint{1.072cm}{1.371cm}}
\pgfpathcurveto{\pgfqpoint{1.072cm}{1.335cm}}{\pgfqpoint{1.087cm}{1.3cm}}{\pgfqpoint{1.112cm}{1.274cm}}
\pgfpathcurveto{\pgfqpoint{1.138cm}{1.249cm}}{\pgfqpoint{1.172cm}{1.234cm}}{\pgfqpoint{1.209cm}{1.234cm}}
\pgfpathcurveto{\pgfqpoint{1.245cm}{1.234cm}}{\pgfqpoint{1.28cm}{1.249cm}}{\pgfqpoint{1.305cm}{1.274cm}}
\pgfpathcurveto{\pgfqpoint{1.331cm}{1.3cm}}{\pgfqpoint{1.345cm}{1.335cm}}{\pgfqpoint{1.345cm}{1.371cm}}
\pgfusepath{fill}
\begin{pgfscope}
\pgfsetdash{}{0cm}
\pgfsetlinewidth{0.818mm}
\pgfsetroundcap
\pgfsetmiterlimit{4.0}
\pgfpathmoveto{\pgfqpoint{0.682cm}{0.671cm}}
\pgfpathlineto{\pgfqpoint{0.682cm}{0.042cm}}
\pgfusepath{stroke}
\end{pgfscope}
\end{pgfscope}
\end{pgfscope}
\end{pgfscope}
\end{tikzpicture}}}+\phi+\psi)\prec\llbracket X^{2}\rrbracket\Big)\\
&\quad+3\llbracket X^2 \rrbracket\preccurlyeq(-X^{\!\resizebox{0.6em}{!}{
\begin{tikzpicture}
\pgfpathmoveto{\pgfqpoint{0cm}{-0.035cm}}
\pgfpathlineto{\pgfqpoint{1.376cm}{-0.035cm}}
\pgfpathlineto{\pgfqpoint{1.376cm}{1.552cm}}
\pgfpathlineto{\pgfqpoint{0cm}{1.552cm}}
\pgfpathclose
\pgfusepath{clip}
\begin{pgfscope}
\begin{pgfscope}
\pgfpathmoveto{\pgfqpoint{0cm}{-0.035cm}}
\pgfpathlineto{\pgfqpoint{1.376cm}{-0.035cm}}
\pgfpathlineto{\pgfqpoint{1.376cm}{1.552cm}}
\pgfpathlineto{\pgfqpoint{0cm}{1.552cm}}
\pgfpathclose
\pgfusepath{clip}
\begin{pgfscope}
\begin{pgfscope}
\pgfsetdash{}{0cm}
\pgfsetlinewidth{0.818mm}
\pgfsetroundcap
\pgfsetroundjoin
\pgfsetmiterlimit{7.0}
\definecolor{eps2pgf_color}{gray}{0}\pgfsetstrokecolor{eps2pgf_color}\pgfsetfillcolor{eps2pgf_color}
\pgfpathmoveto{\pgfqpoint{0.117cm}{1.421cm}}
\pgfpathlineto{\pgfqpoint{0.682cm}{0.671cm}}
\pgfpathlineto{\pgfqpoint{1.246cm}{1.421cm}}
\pgfusepath{stroke}
\end{pgfscope}
\definecolor{eps2pgf_color}{gray}{0}\pgfsetstrokecolor{eps2pgf_color}\pgfsetfillcolor{eps2pgf_color}
\pgfpathmoveto{\pgfqpoint{0.273cm}{1.395cm}}
\pgfpathcurveto{\pgfqpoint{0.273cm}{1.432cm}}{\pgfqpoint{0.259cm}{1.467cm}}{\pgfqpoint{0.233cm}{1.492cm}}
\pgfpathcurveto{\pgfqpoint{0.207cm}{1.518cm}}{\pgfqpoint{0.173cm}{1.532cm}}{\pgfqpoint{0.137cm}{1.532cm}}
\pgfpathcurveto{\pgfqpoint{0.1cm}{1.532cm}}{\pgfqpoint{0.066cm}{1.518cm}}{\pgfqpoint{0.04cm}{1.492cm}}
\pgfpathcurveto{\pgfqpoint{0.014cm}{1.467cm}}{\pgfqpoint{0cm}{1.432cm}}{\pgfqpoint{0cm}{1.395cm}}
\pgfpathcurveto{\pgfqpoint{0cm}{1.359cm}}{\pgfqpoint{0.014cm}{1.324cm}}{\pgfqpoint{0.04cm}{1.299cm}}
\pgfpathcurveto{\pgfqpoint{0.066cm}{1.273cm}}{\pgfqpoint{0.1cm}{1.258cm}}{\pgfqpoint{0.137cm}{1.258cm}}
\pgfpathcurveto{\pgfqpoint{0.173cm}{1.258cm}}{\pgfqpoint{0.207cm}{1.273cm}}{\pgfqpoint{0.233cm}{1.299cm}}
\pgfpathcurveto{\pgfqpoint{0.259cm}{1.324cm}}{\pgfqpoint{0.273cm}{1.359cm}}{\pgfqpoint{0.273cm}{1.395cm}}
\pgfusepath{fill}
\begin{pgfscope}
\pgfsetdash{}{0cm}
\pgfsetlinewidth{0.818mm}
\pgfsetmiterlimit{7.0}
\pgfpathmoveto{\pgfqpoint{0.682cm}{0.671cm}}
\pgfpathlineto{\pgfqpoint{0.679cm}{1.418cm}}
\pgfusepath{stroke}
\end{pgfscope}
\pgfpathmoveto{\pgfqpoint{0.815cm}{1.399cm}}
\pgfpathcurveto{\pgfqpoint{0.815cm}{1.435cm}}{\pgfqpoint{0.801cm}{1.47cm}}{\pgfqpoint{0.775cm}{1.496cm}}
\pgfpathcurveto{\pgfqpoint{0.75cm}{1.521cm}}{\pgfqpoint{0.715cm}{1.536cm}}{\pgfqpoint{0.679cm}{1.536cm}}
\pgfpathcurveto{\pgfqpoint{0.643cm}{1.536cm}}{\pgfqpoint{0.608cm}{1.521cm}}{\pgfqpoint{0.582cm}{1.496cm}}
\pgfpathcurveto{\pgfqpoint{0.557cm}{1.47cm}}{\pgfqpoint{0.542cm}{1.435cm}}{\pgfqpoint{0.542cm}{1.399cm}}
\pgfpathcurveto{\pgfqpoint{0.542cm}{1.363cm}}{\pgfqpoint{0.557cm}{1.328cm}}{\pgfqpoint{0.582cm}{1.302cm}}
\pgfpathcurveto{\pgfqpoint{0.608cm}{1.276cm}}{\pgfqpoint{0.643cm}{1.262cm}}{\pgfqpoint{0.679cm}{1.262cm}}
\pgfpathcurveto{\pgfqpoint{0.715cm}{1.262cm}}{\pgfqpoint{0.75cm}{1.276cm}}{\pgfqpoint{0.775cm}{1.302cm}}
\pgfpathcurveto{\pgfqpoint{0.801cm}{1.328cm}}{\pgfqpoint{0.815cm}{1.363cm}}{\pgfqpoint{0.815cm}{1.399cm}}
\pgfusepath{fill}
\pgfpathmoveto{\pgfqpoint{1.345cm}{1.371cm}}
\pgfpathcurveto{\pgfqpoint{1.345cm}{1.408cm}}{\pgfqpoint{1.331cm}{1.442cm}}{\pgfqpoint{1.305cm}{1.468cm}}
\pgfpathcurveto{\pgfqpoint{1.28cm}{1.494cm}}{\pgfqpoint{1.245cm}{1.508cm}}{\pgfqpoint{1.209cm}{1.508cm}}
\pgfpathcurveto{\pgfqpoint{1.172cm}{1.508cm}}{\pgfqpoint{1.138cm}{1.494cm}}{\pgfqpoint{1.112cm}{1.468cm}}
\pgfpathcurveto{\pgfqpoint{1.087cm}{1.442cm}}{\pgfqpoint{1.072cm}{1.408cm}}{\pgfqpoint{1.072cm}{1.371cm}}
\pgfpathcurveto{\pgfqpoint{1.072cm}{1.335cm}}{\pgfqpoint{1.087cm}{1.3cm}}{\pgfqpoint{1.112cm}{1.274cm}}
\pgfpathcurveto{\pgfqpoint{1.138cm}{1.249cm}}{\pgfqpoint{1.172cm}{1.234cm}}{\pgfqpoint{1.209cm}{1.234cm}}
\pgfpathcurveto{\pgfqpoint{1.245cm}{1.234cm}}{\pgfqpoint{1.28cm}{1.249cm}}{\pgfqpoint{1.305cm}{1.274cm}}
\pgfpathcurveto{\pgfqpoint{1.331cm}{1.3cm}}{\pgfqpoint{1.345cm}{1.335cm}}{\pgfqpoint{1.345cm}{1.371cm}}
\pgfusepath{fill}
\begin{pgfscope}
\pgfsetdash{}{0cm}
\pgfsetlinewidth{0.818mm}
\pgfsetroundcap
\pgfsetmiterlimit{4.0}
\pgfpathmoveto{\pgfqpoint{0.682cm}{0.671cm}}
\pgfpathlineto{\pgfqpoint{0.682cm}{0.042cm}}
\pgfusepath{stroke}
\end{pgfscope}
\end{pgfscope}
\end{pgfscope}
\end{pgfscope}
\end{tikzpicture}}}+\phi+\psi)-3[\LL,(-X^{\!\resizebox{0.6em}{!}{
\begin{tikzpicture}
\pgfpathmoveto{\pgfqpoint{0cm}{-0.035cm}}
\pgfpathlineto{\pgfqpoint{1.376cm}{-0.035cm}}
\pgfpathlineto{\pgfqpoint{1.376cm}{1.552cm}}
\pgfpathlineto{\pgfqpoint{0cm}{1.552cm}}
\pgfpathclose
\pgfusepath{clip}
\begin{pgfscope}
\begin{pgfscope}
\pgfpathmoveto{\pgfqpoint{0cm}{-0.035cm}}
\pgfpathlineto{\pgfqpoint{1.376cm}{-0.035cm}}
\pgfpathlineto{\pgfqpoint{1.376cm}{1.552cm}}
\pgfpathlineto{\pgfqpoint{0cm}{1.552cm}}
\pgfpathclose
\pgfusepath{clip}
\begin{pgfscope}
\begin{pgfscope}
\pgfsetdash{}{0cm}
\pgfsetlinewidth{0.818mm}
\pgfsetroundcap
\pgfsetroundjoin
\pgfsetmiterlimit{7.0}
\definecolor{eps2pgf_color}{gray}{0}\pgfsetstrokecolor{eps2pgf_color}\pgfsetfillcolor{eps2pgf_color}
\pgfpathmoveto{\pgfqpoint{0.117cm}{1.421cm}}
\pgfpathlineto{\pgfqpoint{0.682cm}{0.671cm}}
\pgfpathlineto{\pgfqpoint{1.246cm}{1.421cm}}
\pgfusepath{stroke}
\end{pgfscope}
\definecolor{eps2pgf_color}{gray}{0}\pgfsetstrokecolor{eps2pgf_color}\pgfsetfillcolor{eps2pgf_color}
\pgfpathmoveto{\pgfqpoint{0.273cm}{1.395cm}}
\pgfpathcurveto{\pgfqpoint{0.273cm}{1.432cm}}{\pgfqpoint{0.259cm}{1.467cm}}{\pgfqpoint{0.233cm}{1.492cm}}
\pgfpathcurveto{\pgfqpoint{0.207cm}{1.518cm}}{\pgfqpoint{0.173cm}{1.532cm}}{\pgfqpoint{0.137cm}{1.532cm}}
\pgfpathcurveto{\pgfqpoint{0.1cm}{1.532cm}}{\pgfqpoint{0.066cm}{1.518cm}}{\pgfqpoint{0.04cm}{1.492cm}}
\pgfpathcurveto{\pgfqpoint{0.014cm}{1.467cm}}{\pgfqpoint{0cm}{1.432cm}}{\pgfqpoint{0cm}{1.395cm}}
\pgfpathcurveto{\pgfqpoint{0cm}{1.359cm}}{\pgfqpoint{0.014cm}{1.324cm}}{\pgfqpoint{0.04cm}{1.299cm}}
\pgfpathcurveto{\pgfqpoint{0.066cm}{1.273cm}}{\pgfqpoint{0.1cm}{1.258cm}}{\pgfqpoint{0.137cm}{1.258cm}}
\pgfpathcurveto{\pgfqpoint{0.173cm}{1.258cm}}{\pgfqpoint{0.207cm}{1.273cm}}{\pgfqpoint{0.233cm}{1.299cm}}
\pgfpathcurveto{\pgfqpoint{0.259cm}{1.324cm}}{\pgfqpoint{0.273cm}{1.359cm}}{\pgfqpoint{0.273cm}{1.395cm}}
\pgfusepath{fill}
\begin{pgfscope}
\pgfsetdash{}{0cm}
\pgfsetlinewidth{0.818mm}
\pgfsetmiterlimit{7.0}
\pgfpathmoveto{\pgfqpoint{0.682cm}{0.671cm}}
\pgfpathlineto{\pgfqpoint{0.679cm}{1.418cm}}
\pgfusepath{stroke}
\end{pgfscope}
\pgfpathmoveto{\pgfqpoint{0.815cm}{1.399cm}}
\pgfpathcurveto{\pgfqpoint{0.815cm}{1.435cm}}{\pgfqpoint{0.801cm}{1.47cm}}{\pgfqpoint{0.775cm}{1.496cm}}
\pgfpathcurveto{\pgfqpoint{0.75cm}{1.521cm}}{\pgfqpoint{0.715cm}{1.536cm}}{\pgfqpoint{0.679cm}{1.536cm}}
\pgfpathcurveto{\pgfqpoint{0.643cm}{1.536cm}}{\pgfqpoint{0.608cm}{1.521cm}}{\pgfqpoint{0.582cm}{1.496cm}}
\pgfpathcurveto{\pgfqpoint{0.557cm}{1.47cm}}{\pgfqpoint{0.542cm}{1.435cm}}{\pgfqpoint{0.542cm}{1.399cm}}
\pgfpathcurveto{\pgfqpoint{0.542cm}{1.363cm}}{\pgfqpoint{0.557cm}{1.328cm}}{\pgfqpoint{0.582cm}{1.302cm}}
\pgfpathcurveto{\pgfqpoint{0.608cm}{1.276cm}}{\pgfqpoint{0.643cm}{1.262cm}}{\pgfqpoint{0.679cm}{1.262cm}}
\pgfpathcurveto{\pgfqpoint{0.715cm}{1.262cm}}{\pgfqpoint{0.75cm}{1.276cm}}{\pgfqpoint{0.775cm}{1.302cm}}
\pgfpathcurveto{\pgfqpoint{0.801cm}{1.328cm}}{\pgfqpoint{0.815cm}{1.363cm}}{\pgfqpoint{0.815cm}{1.399cm}}
\pgfusepath{fill}
\pgfpathmoveto{\pgfqpoint{1.345cm}{1.371cm}}
\pgfpathcurveto{\pgfqpoint{1.345cm}{1.408cm}}{\pgfqpoint{1.331cm}{1.442cm}}{\pgfqpoint{1.305cm}{1.468cm}}
\pgfpathcurveto{\pgfqpoint{1.28cm}{1.494cm}}{\pgfqpoint{1.245cm}{1.508cm}}{\pgfqpoint{1.209cm}{1.508cm}}
\pgfpathcurveto{\pgfqpoint{1.172cm}{1.508cm}}{\pgfqpoint{1.138cm}{1.494cm}}{\pgfqpoint{1.112cm}{1.468cm}}
\pgfpathcurveto{\pgfqpoint{1.087cm}{1.442cm}}{\pgfqpoint{1.072cm}{1.408cm}}{\pgfqpoint{1.072cm}{1.371cm}}
\pgfpathcurveto{\pgfqpoint{1.072cm}{1.335cm}}{\pgfqpoint{1.087cm}{1.3cm}}{\pgfqpoint{1.112cm}{1.274cm}}
\pgfpathcurveto{\pgfqpoint{1.138cm}{1.249cm}}{\pgfqpoint{1.172cm}{1.234cm}}{\pgfqpoint{1.209cm}{1.234cm}}
\pgfpathcurveto{\pgfqpoint{1.245cm}{1.234cm}}{\pgfqpoint{1.28cm}{1.249cm}}{\pgfqpoint{1.305cm}{1.274cm}}
\pgfpathcurveto{\pgfqpoint{1.331cm}{1.3cm}}{\pgfqpoint{1.345cm}{1.335cm}}{\pgfqpoint{1.345cm}{1.371cm}}
\pgfusepath{fill}
\begin{pgfscope}
\pgfsetdash{}{0cm}
\pgfsetlinewidth{0.818mm}
\pgfsetroundcap
\pgfsetmiterlimit{4.0}
\pgfpathmoveto{\pgfqpoint{0.682cm}{0.671cm}}
\pgfpathlineto{\pgfqpoint{0.682cm}{0.042cm}}
\pgfusepath{stroke}
\end{pgfscope}
\end{pgfscope}
\end{pgfscope}
\end{pgfscope}
\end{tikzpicture}}}+\phi+\psi)\Prec]X^{\!\resizebox{0.6em}{!}{
\begin{tikzpicture}
\pgfpathmoveto{\pgfqpoint{0cm}{0cm}}
\pgfpathlineto{\pgfqpoint{1.376cm}{0cm}}
\pgfpathlineto{\pgfqpoint{1.376cm}{1.588cm}}
\pgfpathlineto{\pgfqpoint{0cm}{1.588cm}}
\pgfpathclose
\pgfusepath{clip}
\begin{pgfscope}
\begin{pgfscope}
\pgfpathmoveto{\pgfqpoint{0cm}{0cm}}
\pgfpathlineto{\pgfqpoint{1.376cm}{0cm}}
\pgfpathlineto{\pgfqpoint{1.376cm}{1.588cm}}
\pgfpathlineto{\pgfqpoint{0cm}{1.588cm}}
\pgfpathclose
\pgfusepath{clip}
\begin{pgfscope}
\begin{pgfscope}
\definecolor{eps2pgf_color}{gray}{0.976471}\pgfsetstrokecolor{eps2pgf_color}\pgfsetfillcolor{eps2pgf_color}
\pgfpathmoveto{\pgfqpoint{0cm}{0cm}}
\pgfpathlineto{\pgfqpoint{1.376cm}{0cm}}
\pgfpathlineto{\pgfqpoint{1.376cm}{1.588cm}}
\pgfpathlineto{\pgfqpoint{0cm}{1.588cm}}
\pgfpathclose
\pgfusepath{fill}
\end{pgfscope}
\begin{pgfscope}
\pgfsetdash{}{0cm}
\pgfsetlinewidth{0.818mm}
\pgfsetroundcap
\pgfsetroundjoin
\pgfsetmiterlimit{7.0}
\definecolor{eps2pgf_color}{gray}{0}\pgfsetstrokecolor{eps2pgf_color}\pgfsetfillcolor{eps2pgf_color}
\pgfpathmoveto{\pgfqpoint{0.117cm}{1.476cm}}
\pgfpathlineto{\pgfqpoint{0.682cm}{0.726cm}}
\pgfpathlineto{\pgfqpoint{1.246cm}{1.476cm}}
\pgfusepath{stroke}
\end{pgfscope}
\definecolor{eps2pgf_color}{gray}{0}\pgfsetstrokecolor{eps2pgf_color}\pgfsetfillcolor{eps2pgf_color}
\pgfpathmoveto{\pgfqpoint{0.273cm}{1.451cm}}
\pgfpathcurveto{\pgfqpoint{0.273cm}{1.487cm}}{\pgfqpoint{0.259cm}{1.522cm}}{\pgfqpoint{0.233cm}{1.547cm}}
\pgfpathcurveto{\pgfqpoint{0.207cm}{1.573cm}}{\pgfqpoint{0.173cm}{1.588cm}}{\pgfqpoint{0.137cm}{1.588cm}}
\pgfpathcurveto{\pgfqpoint{0.1cm}{1.588cm}}{\pgfqpoint{0.066cm}{1.573cm}}{\pgfqpoint{0.04cm}{1.547cm}}
\pgfpathcurveto{\pgfqpoint{0.014cm}{1.522cm}}{\pgfqpoint{0cm}{1.487cm}}{\pgfqpoint{0cm}{1.451cm}}
\pgfpathcurveto{\pgfqpoint{0cm}{1.414cm}}{\pgfqpoint{0.014cm}{1.379cm}}{\pgfqpoint{0.04cm}{1.354cm}}
\pgfpathcurveto{\pgfqpoint{0.066cm}{1.328cm}}{\pgfqpoint{0.1cm}{1.314cm}}{\pgfqpoint{0.137cm}{1.314cm}}
\pgfpathcurveto{\pgfqpoint{0.173cm}{1.314cm}}{\pgfqpoint{0.207cm}{1.328cm}}{\pgfqpoint{0.233cm}{1.354cm}}
\pgfpathcurveto{\pgfqpoint{0.259cm}{1.379cm}}{\pgfqpoint{0.273cm}{1.414cm}}{\pgfqpoint{0.273cm}{1.451cm}}
\pgfusepath{fill}
\pgfpathmoveto{\pgfqpoint{1.345cm}{1.426cm}}
\pgfpathcurveto{\pgfqpoint{1.345cm}{1.463cm}}{\pgfqpoint{1.331cm}{1.497cm}}{\pgfqpoint{1.305cm}{1.523cm}}
\pgfpathcurveto{\pgfqpoint{1.28cm}{1.549cm}}{\pgfqpoint{1.245cm}{1.563cm}}{\pgfqpoint{1.209cm}{1.563cm}}
\pgfpathcurveto{\pgfqpoint{1.172cm}{1.563cm}}{\pgfqpoint{1.138cm}{1.549cm}}{\pgfqpoint{1.112cm}{1.523cm}}
\pgfpathcurveto{\pgfqpoint{1.087cm}{1.497cm}}{\pgfqpoint{1.072cm}{1.463cm}}{\pgfqpoint{1.072cm}{1.426cm}}
\pgfpathcurveto{\pgfqpoint{1.072cm}{1.39cm}}{\pgfqpoint{1.087cm}{1.355cm}}{\pgfqpoint{1.112cm}{1.329cm}}
\pgfpathcurveto{\pgfqpoint{1.138cm}{1.304cm}}{\pgfqpoint{1.172cm}{1.289cm}}{\pgfqpoint{1.209cm}{1.289cm}}
\pgfpathcurveto{\pgfqpoint{1.245cm}{1.289cm}}{\pgfqpoint{1.28cm}{1.304cm}}{\pgfqpoint{1.305cm}{1.329cm}}
\pgfpathcurveto{\pgfqpoint{1.331cm}{1.355cm}}{\pgfqpoint{1.345cm}{1.39cm}}{\pgfqpoint{1.345cm}{1.426cm}}
\pgfusepath{fill}
\begin{pgfscope}
\pgfsetdash{}{0cm}
\pgfsetlinewidth{0.818mm}
\pgfsetroundcap
\pgfsetmiterlimit{4.0}
\pgfpathmoveto{\pgfqpoint{0.682cm}{0.726cm}}
\pgfpathlineto{\pgfqpoint{0.682cm}{0.097cm}}
\pgfusepath{stroke}
\end{pgfscope}
\end{pgfscope}
\end{pgfscope}
\end{pgfscope}
\end{tikzpicture}}}\\
&\quad+3X(-X^{\!\resizebox{0.6em}{!}{
\begin{tikzpicture}
\pgfpathmoveto{\pgfqpoint{0cm}{-0.035cm}}
\pgfpathlineto{\pgfqpoint{1.376cm}{-0.035cm}}
\pgfpathlineto{\pgfqpoint{1.376cm}{1.552cm}}
\pgfpathlineto{\pgfqpoint{0cm}{1.552cm}}
\pgfpathclose
\pgfusepath{clip}
\begin{pgfscope}
\begin{pgfscope}
\pgfpathmoveto{\pgfqpoint{0cm}{-0.035cm}}
\pgfpathlineto{\pgfqpoint{1.376cm}{-0.035cm}}
\pgfpathlineto{\pgfqpoint{1.376cm}{1.552cm}}
\pgfpathlineto{\pgfqpoint{0cm}{1.552cm}}
\pgfpathclose
\pgfusepath{clip}
\begin{pgfscope}
\begin{pgfscope}
\pgfsetdash{}{0cm}
\pgfsetlinewidth{0.818mm}
\pgfsetroundcap
\pgfsetroundjoin
\pgfsetmiterlimit{7.0}
\definecolor{eps2pgf_color}{gray}{0}\pgfsetstrokecolor{eps2pgf_color}\pgfsetfillcolor{eps2pgf_color}
\pgfpathmoveto{\pgfqpoint{0.117cm}{1.421cm}}
\pgfpathlineto{\pgfqpoint{0.682cm}{0.671cm}}
\pgfpathlineto{\pgfqpoint{1.246cm}{1.421cm}}
\pgfusepath{stroke}
\end{pgfscope}
\definecolor{eps2pgf_color}{gray}{0}\pgfsetstrokecolor{eps2pgf_color}\pgfsetfillcolor{eps2pgf_color}
\pgfpathmoveto{\pgfqpoint{0.273cm}{1.395cm}}
\pgfpathcurveto{\pgfqpoint{0.273cm}{1.432cm}}{\pgfqpoint{0.259cm}{1.467cm}}{\pgfqpoint{0.233cm}{1.492cm}}
\pgfpathcurveto{\pgfqpoint{0.207cm}{1.518cm}}{\pgfqpoint{0.173cm}{1.532cm}}{\pgfqpoint{0.137cm}{1.532cm}}
\pgfpathcurveto{\pgfqpoint{0.1cm}{1.532cm}}{\pgfqpoint{0.066cm}{1.518cm}}{\pgfqpoint{0.04cm}{1.492cm}}
\pgfpathcurveto{\pgfqpoint{0.014cm}{1.467cm}}{\pgfqpoint{0cm}{1.432cm}}{\pgfqpoint{0cm}{1.395cm}}
\pgfpathcurveto{\pgfqpoint{0cm}{1.359cm}}{\pgfqpoint{0.014cm}{1.324cm}}{\pgfqpoint{0.04cm}{1.299cm}}
\pgfpathcurveto{\pgfqpoint{0.066cm}{1.273cm}}{\pgfqpoint{0.1cm}{1.258cm}}{\pgfqpoint{0.137cm}{1.258cm}}
\pgfpathcurveto{\pgfqpoint{0.173cm}{1.258cm}}{\pgfqpoint{0.207cm}{1.273cm}}{\pgfqpoint{0.233cm}{1.299cm}}
\pgfpathcurveto{\pgfqpoint{0.259cm}{1.324cm}}{\pgfqpoint{0.273cm}{1.359cm}}{\pgfqpoint{0.273cm}{1.395cm}}
\pgfusepath{fill}
\begin{pgfscope}
\pgfsetdash{}{0cm}
\pgfsetlinewidth{0.818mm}
\pgfsetmiterlimit{7.0}
\pgfpathmoveto{\pgfqpoint{0.682cm}{0.671cm}}
\pgfpathlineto{\pgfqpoint{0.679cm}{1.418cm}}
\pgfusepath{stroke}
\end{pgfscope}
\pgfpathmoveto{\pgfqpoint{0.815cm}{1.399cm}}
\pgfpathcurveto{\pgfqpoint{0.815cm}{1.435cm}}{\pgfqpoint{0.801cm}{1.47cm}}{\pgfqpoint{0.775cm}{1.496cm}}
\pgfpathcurveto{\pgfqpoint{0.75cm}{1.521cm}}{\pgfqpoint{0.715cm}{1.536cm}}{\pgfqpoint{0.679cm}{1.536cm}}
\pgfpathcurveto{\pgfqpoint{0.643cm}{1.536cm}}{\pgfqpoint{0.608cm}{1.521cm}}{\pgfqpoint{0.582cm}{1.496cm}}
\pgfpathcurveto{\pgfqpoint{0.557cm}{1.47cm}}{\pgfqpoint{0.542cm}{1.435cm}}{\pgfqpoint{0.542cm}{1.399cm}}
\pgfpathcurveto{\pgfqpoint{0.542cm}{1.363cm}}{\pgfqpoint{0.557cm}{1.328cm}}{\pgfqpoint{0.582cm}{1.302cm}}
\pgfpathcurveto{\pgfqpoint{0.608cm}{1.276cm}}{\pgfqpoint{0.643cm}{1.262cm}}{\pgfqpoint{0.679cm}{1.262cm}}
\pgfpathcurveto{\pgfqpoint{0.715cm}{1.262cm}}{\pgfqpoint{0.75cm}{1.276cm}}{\pgfqpoint{0.775cm}{1.302cm}}
\pgfpathcurveto{\pgfqpoint{0.801cm}{1.328cm}}{\pgfqpoint{0.815cm}{1.363cm}}{\pgfqpoint{0.815cm}{1.399cm}}
\pgfusepath{fill}
\pgfpathmoveto{\pgfqpoint{1.345cm}{1.371cm}}
\pgfpathcurveto{\pgfqpoint{1.345cm}{1.408cm}}{\pgfqpoint{1.331cm}{1.442cm}}{\pgfqpoint{1.305cm}{1.468cm}}
\pgfpathcurveto{\pgfqpoint{1.28cm}{1.494cm}}{\pgfqpoint{1.245cm}{1.508cm}}{\pgfqpoint{1.209cm}{1.508cm}}
\pgfpathcurveto{\pgfqpoint{1.172cm}{1.508cm}}{\pgfqpoint{1.138cm}{1.494cm}}{\pgfqpoint{1.112cm}{1.468cm}}
\pgfpathcurveto{\pgfqpoint{1.087cm}{1.442cm}}{\pgfqpoint{1.072cm}{1.408cm}}{\pgfqpoint{1.072cm}{1.371cm}}
\pgfpathcurveto{\pgfqpoint{1.072cm}{1.335cm}}{\pgfqpoint{1.087cm}{1.3cm}}{\pgfqpoint{1.112cm}{1.274cm}}
\pgfpathcurveto{\pgfqpoint{1.138cm}{1.249cm}}{\pgfqpoint{1.172cm}{1.234cm}}{\pgfqpoint{1.209cm}{1.234cm}}
\pgfpathcurveto{\pgfqpoint{1.245cm}{1.234cm}}{\pgfqpoint{1.28cm}{1.249cm}}{\pgfqpoint{1.305cm}{1.274cm}}
\pgfpathcurveto{\pgfqpoint{1.331cm}{1.3cm}}{\pgfqpoint{1.345cm}{1.335cm}}{\pgfqpoint{1.345cm}{1.371cm}}
\pgfusepath{fill}
\begin{pgfscope}
\pgfsetdash{}{0cm}
\pgfsetlinewidth{0.818mm}
\pgfsetroundcap
\pgfsetmiterlimit{4.0}
\pgfpathmoveto{\pgfqpoint{0.682cm}{0.671cm}}
\pgfpathlineto{\pgfqpoint{0.682cm}{0.042cm}}
\pgfusepath{stroke}
\end{pgfscope}
\end{pgfscope}
\end{pgfscope}
\end{pgfscope}
\end{tikzpicture}}}+\phi+\psi)^2+(-X^{\!\resizebox{0.6em}{!}{
\begin{tikzpicture}
\pgfpathmoveto{\pgfqpoint{0cm}{-0.035cm}}
\pgfpathlineto{\pgfqpoint{1.376cm}{-0.035cm}}
\pgfpathlineto{\pgfqpoint{1.376cm}{1.552cm}}
\pgfpathlineto{\pgfqpoint{0cm}{1.552cm}}
\pgfpathclose
\pgfusepath{clip}
\begin{pgfscope}
\begin{pgfscope}
\pgfpathmoveto{\pgfqpoint{0cm}{-0.035cm}}
\pgfpathlineto{\pgfqpoint{1.376cm}{-0.035cm}}
\pgfpathlineto{\pgfqpoint{1.376cm}{1.552cm}}
\pgfpathlineto{\pgfqpoint{0cm}{1.552cm}}
\pgfpathclose
\pgfusepath{clip}
\begin{pgfscope}
\begin{pgfscope}
\pgfsetdash{}{0cm}
\pgfsetlinewidth{0.818mm}
\pgfsetroundcap
\pgfsetroundjoin
\pgfsetmiterlimit{7.0}
\definecolor{eps2pgf_color}{gray}{0}\pgfsetstrokecolor{eps2pgf_color}\pgfsetfillcolor{eps2pgf_color}
\pgfpathmoveto{\pgfqpoint{0.117cm}{1.421cm}}
\pgfpathlineto{\pgfqpoint{0.682cm}{0.671cm}}
\pgfpathlineto{\pgfqpoint{1.246cm}{1.421cm}}
\pgfusepath{stroke}
\end{pgfscope}
\definecolor{eps2pgf_color}{gray}{0}\pgfsetstrokecolor{eps2pgf_color}\pgfsetfillcolor{eps2pgf_color}
\pgfpathmoveto{\pgfqpoint{0.273cm}{1.395cm}}
\pgfpathcurveto{\pgfqpoint{0.273cm}{1.432cm}}{\pgfqpoint{0.259cm}{1.467cm}}{\pgfqpoint{0.233cm}{1.492cm}}
\pgfpathcurveto{\pgfqpoint{0.207cm}{1.518cm}}{\pgfqpoint{0.173cm}{1.532cm}}{\pgfqpoint{0.137cm}{1.532cm}}
\pgfpathcurveto{\pgfqpoint{0.1cm}{1.532cm}}{\pgfqpoint{0.066cm}{1.518cm}}{\pgfqpoint{0.04cm}{1.492cm}}
\pgfpathcurveto{\pgfqpoint{0.014cm}{1.467cm}}{\pgfqpoint{0cm}{1.432cm}}{\pgfqpoint{0cm}{1.395cm}}
\pgfpathcurveto{\pgfqpoint{0cm}{1.359cm}}{\pgfqpoint{0.014cm}{1.324cm}}{\pgfqpoint{0.04cm}{1.299cm}}
\pgfpathcurveto{\pgfqpoint{0.066cm}{1.273cm}}{\pgfqpoint{0.1cm}{1.258cm}}{\pgfqpoint{0.137cm}{1.258cm}}
\pgfpathcurveto{\pgfqpoint{0.173cm}{1.258cm}}{\pgfqpoint{0.207cm}{1.273cm}}{\pgfqpoint{0.233cm}{1.299cm}}
\pgfpathcurveto{\pgfqpoint{0.259cm}{1.324cm}}{\pgfqpoint{0.273cm}{1.359cm}}{\pgfqpoint{0.273cm}{1.395cm}}
\pgfusepath{fill}
\begin{pgfscope}
\pgfsetdash{}{0cm}
\pgfsetlinewidth{0.818mm}
\pgfsetmiterlimit{7.0}
\pgfpathmoveto{\pgfqpoint{0.682cm}{0.671cm}}
\pgfpathlineto{\pgfqpoint{0.679cm}{1.418cm}}
\pgfusepath{stroke}
\end{pgfscope}
\pgfpathmoveto{\pgfqpoint{0.815cm}{1.399cm}}
\pgfpathcurveto{\pgfqpoint{0.815cm}{1.435cm}}{\pgfqpoint{0.801cm}{1.47cm}}{\pgfqpoint{0.775cm}{1.496cm}}
\pgfpathcurveto{\pgfqpoint{0.75cm}{1.521cm}}{\pgfqpoint{0.715cm}{1.536cm}}{\pgfqpoint{0.679cm}{1.536cm}}
\pgfpathcurveto{\pgfqpoint{0.643cm}{1.536cm}}{\pgfqpoint{0.608cm}{1.521cm}}{\pgfqpoint{0.582cm}{1.496cm}}
\pgfpathcurveto{\pgfqpoint{0.557cm}{1.47cm}}{\pgfqpoint{0.542cm}{1.435cm}}{\pgfqpoint{0.542cm}{1.399cm}}
\pgfpathcurveto{\pgfqpoint{0.542cm}{1.363cm}}{\pgfqpoint{0.557cm}{1.328cm}}{\pgfqpoint{0.582cm}{1.302cm}}
\pgfpathcurveto{\pgfqpoint{0.608cm}{1.276cm}}{\pgfqpoint{0.643cm}{1.262cm}}{\pgfqpoint{0.679cm}{1.262cm}}
\pgfpathcurveto{\pgfqpoint{0.715cm}{1.262cm}}{\pgfqpoint{0.75cm}{1.276cm}}{\pgfqpoint{0.775cm}{1.302cm}}
\pgfpathcurveto{\pgfqpoint{0.801cm}{1.328cm}}{\pgfqpoint{0.815cm}{1.363cm}}{\pgfqpoint{0.815cm}{1.399cm}}
\pgfusepath{fill}
\pgfpathmoveto{\pgfqpoint{1.345cm}{1.371cm}}
\pgfpathcurveto{\pgfqpoint{1.345cm}{1.408cm}}{\pgfqpoint{1.331cm}{1.442cm}}{\pgfqpoint{1.305cm}{1.468cm}}
\pgfpathcurveto{\pgfqpoint{1.28cm}{1.494cm}}{\pgfqpoint{1.245cm}{1.508cm}}{\pgfqpoint{1.209cm}{1.508cm}}
\pgfpathcurveto{\pgfqpoint{1.172cm}{1.508cm}}{\pgfqpoint{1.138cm}{1.494cm}}{\pgfqpoint{1.112cm}{1.468cm}}
\pgfpathcurveto{\pgfqpoint{1.087cm}{1.442cm}}{\pgfqpoint{1.072cm}{1.408cm}}{\pgfqpoint{1.072cm}{1.371cm}}
\pgfpathcurveto{\pgfqpoint{1.072cm}{1.335cm}}{\pgfqpoint{1.087cm}{1.3cm}}{\pgfqpoint{1.112cm}{1.274cm}}
\pgfpathcurveto{\pgfqpoint{1.138cm}{1.249cm}}{\pgfqpoint{1.172cm}{1.234cm}}{\pgfqpoint{1.209cm}{1.234cm}}
\pgfpathcurveto{\pgfqpoint{1.245cm}{1.234cm}}{\pgfqpoint{1.28cm}{1.249cm}}{\pgfqpoint{1.305cm}{1.274cm}}
\pgfpathcurveto{\pgfqpoint{1.331cm}{1.3cm}}{\pgfqpoint{1.345cm}{1.335cm}}{\pgfqpoint{1.345cm}{1.371cm}}
\pgfusepath{fill}
\begin{pgfscope}
\pgfsetdash{}{0cm}
\pgfsetlinewidth{0.818mm}
\pgfsetroundcap
\pgfsetmiterlimit{4.0}
\pgfpathmoveto{\pgfqpoint{0.682cm}{0.671cm}}
\pgfpathlineto{\pgfqpoint{0.682cm}{0.042cm}}
\pgfusepath{stroke}
\end{pgfscope}
\end{pgfscope}
\end{pgfscope}
\end{pgfscope}
\end{tikzpicture}}}+\phi+\psi)^3+3b\varphi.
\end{aligned}
\end{align}
Next, we proceed with the same decomposition into regular (blue) and irregular (magenta) part as in Section \ref{s:phi}. It leads  to
\begin{align}\label{eq:two43}
\begin{aligned}
& \LL \phi + \rmm{\Phi} = 0,\quad\phi(0)=\phi_{0}=\varphi_{0},\\
&  \LL \psi + \psi^3 + \rmb{\tilde\Psi} = 0,\quad\psi(0)=0,
  \end{aligned}
 \end{align}
 where
 \begin{equation}\label{eq:two43aa}
 \rmb{\tilde\Psi}:=  \rmb{\Psi} -9\llbracket X^{2}\rrbracket \circ[(-X^{\!\resizebox{0.6em}{!}{
\begin{tikzpicture}
\pgfpathmoveto{\pgfqpoint{0cm}{-0.035cm}}
\pgfpathlineto{\pgfqpoint{1.376cm}{-0.035cm}}
\pgfpathlineto{\pgfqpoint{1.376cm}{1.552cm}}
\pgfpathlineto{\pgfqpoint{0cm}{1.552cm}}
\pgfpathclose
\pgfusepath{clip}
\begin{pgfscope}
\begin{pgfscope}
\pgfpathmoveto{\pgfqpoint{0cm}{-0.035cm}}
\pgfpathlineto{\pgfqpoint{1.376cm}{-0.035cm}}
\pgfpathlineto{\pgfqpoint{1.376cm}{1.552cm}}
\pgfpathlineto{\pgfqpoint{0cm}{1.552cm}}
\pgfpathclose
\pgfusepath{clip}
\begin{pgfscope}
\begin{pgfscope}
\pgfsetdash{}{0cm}
\pgfsetlinewidth{0.818mm}
\pgfsetroundcap
\pgfsetroundjoin
\pgfsetmiterlimit{7.0}
\definecolor{eps2pgf_color}{gray}{0}\pgfsetstrokecolor{eps2pgf_color}\pgfsetfillcolor{eps2pgf_color}
\pgfpathmoveto{\pgfqpoint{0.117cm}{1.421cm}}
\pgfpathlineto{\pgfqpoint{0.682cm}{0.671cm}}
\pgfpathlineto{\pgfqpoint{1.246cm}{1.421cm}}
\pgfusepath{stroke}
\end{pgfscope}
\definecolor{eps2pgf_color}{gray}{0}\pgfsetstrokecolor{eps2pgf_color}\pgfsetfillcolor{eps2pgf_color}
\pgfpathmoveto{\pgfqpoint{0.273cm}{1.395cm}}
\pgfpathcurveto{\pgfqpoint{0.273cm}{1.432cm}}{\pgfqpoint{0.259cm}{1.467cm}}{\pgfqpoint{0.233cm}{1.492cm}}
\pgfpathcurveto{\pgfqpoint{0.207cm}{1.518cm}}{\pgfqpoint{0.173cm}{1.532cm}}{\pgfqpoint{0.137cm}{1.532cm}}
\pgfpathcurveto{\pgfqpoint{0.1cm}{1.532cm}}{\pgfqpoint{0.066cm}{1.518cm}}{\pgfqpoint{0.04cm}{1.492cm}}
\pgfpathcurveto{\pgfqpoint{0.014cm}{1.467cm}}{\pgfqpoint{0cm}{1.432cm}}{\pgfqpoint{0cm}{1.395cm}}
\pgfpathcurveto{\pgfqpoint{0cm}{1.359cm}}{\pgfqpoint{0.014cm}{1.324cm}}{\pgfqpoint{0.04cm}{1.299cm}}
\pgfpathcurveto{\pgfqpoint{0.066cm}{1.273cm}}{\pgfqpoint{0.1cm}{1.258cm}}{\pgfqpoint{0.137cm}{1.258cm}}
\pgfpathcurveto{\pgfqpoint{0.173cm}{1.258cm}}{\pgfqpoint{0.207cm}{1.273cm}}{\pgfqpoint{0.233cm}{1.299cm}}
\pgfpathcurveto{\pgfqpoint{0.259cm}{1.324cm}}{\pgfqpoint{0.273cm}{1.359cm}}{\pgfqpoint{0.273cm}{1.395cm}}
\pgfusepath{fill}
\begin{pgfscope}
\pgfsetdash{}{0cm}
\pgfsetlinewidth{0.818mm}
\pgfsetmiterlimit{7.0}
\pgfpathmoveto{\pgfqpoint{0.682cm}{0.671cm}}
\pgfpathlineto{\pgfqpoint{0.679cm}{1.418cm}}
\pgfusepath{stroke}
\end{pgfscope}
\pgfpathmoveto{\pgfqpoint{0.815cm}{1.399cm}}
\pgfpathcurveto{\pgfqpoint{0.815cm}{1.435cm}}{\pgfqpoint{0.801cm}{1.47cm}}{\pgfqpoint{0.775cm}{1.496cm}}
\pgfpathcurveto{\pgfqpoint{0.75cm}{1.521cm}}{\pgfqpoint{0.715cm}{1.536cm}}{\pgfqpoint{0.679cm}{1.536cm}}
\pgfpathcurveto{\pgfqpoint{0.643cm}{1.536cm}}{\pgfqpoint{0.608cm}{1.521cm}}{\pgfqpoint{0.582cm}{1.496cm}}
\pgfpathcurveto{\pgfqpoint{0.557cm}{1.47cm}}{\pgfqpoint{0.542cm}{1.435cm}}{\pgfqpoint{0.542cm}{1.399cm}}
\pgfpathcurveto{\pgfqpoint{0.542cm}{1.363cm}}{\pgfqpoint{0.557cm}{1.328cm}}{\pgfqpoint{0.582cm}{1.302cm}}
\pgfpathcurveto{\pgfqpoint{0.608cm}{1.276cm}}{\pgfqpoint{0.643cm}{1.262cm}}{\pgfqpoint{0.679cm}{1.262cm}}
\pgfpathcurveto{\pgfqpoint{0.715cm}{1.262cm}}{\pgfqpoint{0.75cm}{1.276cm}}{\pgfqpoint{0.775cm}{1.302cm}}
\pgfpathcurveto{\pgfqpoint{0.801cm}{1.328cm}}{\pgfqpoint{0.815cm}{1.363cm}}{\pgfqpoint{0.815cm}{1.399cm}}
\pgfusepath{fill}
\pgfpathmoveto{\pgfqpoint{1.345cm}{1.371cm}}
\pgfpathcurveto{\pgfqpoint{1.345cm}{1.408cm}}{\pgfqpoint{1.331cm}{1.442cm}}{\pgfqpoint{1.305cm}{1.468cm}}
\pgfpathcurveto{\pgfqpoint{1.28cm}{1.494cm}}{\pgfqpoint{1.245cm}{1.508cm}}{\pgfqpoint{1.209cm}{1.508cm}}
\pgfpathcurveto{\pgfqpoint{1.172cm}{1.508cm}}{\pgfqpoint{1.138cm}{1.494cm}}{\pgfqpoint{1.112cm}{1.468cm}}
\pgfpathcurveto{\pgfqpoint{1.087cm}{1.442cm}}{\pgfqpoint{1.072cm}{1.408cm}}{\pgfqpoint{1.072cm}{1.371cm}}
\pgfpathcurveto{\pgfqpoint{1.072cm}{1.335cm}}{\pgfqpoint{1.087cm}{1.3cm}}{\pgfqpoint{1.112cm}{1.274cm}}
\pgfpathcurveto{\pgfqpoint{1.138cm}{1.249cm}}{\pgfqpoint{1.172cm}{1.234cm}}{\pgfqpoint{1.209cm}{1.234cm}}
\pgfpathcurveto{\pgfqpoint{1.245cm}{1.234cm}}{\pgfqpoint{1.28cm}{1.249cm}}{\pgfqpoint{1.305cm}{1.274cm}}
\pgfpathcurveto{\pgfqpoint{1.331cm}{1.3cm}}{\pgfqpoint{1.345cm}{1.335cm}}{\pgfqpoint{1.345cm}{1.371cm}}
\pgfusepath{fill}
\begin{pgfscope}
\pgfsetdash{}{0cm}
\pgfsetlinewidth{0.818mm}
\pgfsetroundcap
\pgfsetmiterlimit{4.0}
\pgfpathmoveto{\pgfqpoint{0.682cm}{0.671cm}}
\pgfpathlineto{\pgfqpoint{0.682cm}{0.042cm}}
\pgfusepath{stroke}
\end{pgfscope}
\end{pgfscope}
\end{pgfscope}
\end{pgfscope}
\end{tikzpicture}}}+\phi+\psi)\precprec X^{\!\resizebox{0.6em}{!}{
\begin{tikzpicture}
\pgfpathmoveto{\pgfqpoint{0cm}{0cm}}
\pgfpathlineto{\pgfqpoint{1.376cm}{0cm}}
\pgfpathlineto{\pgfqpoint{1.376cm}{1.588cm}}
\pgfpathlineto{\pgfqpoint{0cm}{1.588cm}}
\pgfpathclose
\pgfusepath{clip}
\begin{pgfscope}
\begin{pgfscope}
\pgfpathmoveto{\pgfqpoint{0cm}{0cm}}
\pgfpathlineto{\pgfqpoint{1.376cm}{0cm}}
\pgfpathlineto{\pgfqpoint{1.376cm}{1.588cm}}
\pgfpathlineto{\pgfqpoint{0cm}{1.588cm}}
\pgfpathclose
\pgfusepath{clip}
\begin{pgfscope}
\begin{pgfscope}
\definecolor{eps2pgf_color}{gray}{0.976471}\pgfsetstrokecolor{eps2pgf_color}\pgfsetfillcolor{eps2pgf_color}
\pgfpathmoveto{\pgfqpoint{0cm}{0cm}}
\pgfpathlineto{\pgfqpoint{1.376cm}{0cm}}
\pgfpathlineto{\pgfqpoint{1.376cm}{1.588cm}}
\pgfpathlineto{\pgfqpoint{0cm}{1.588cm}}
\pgfpathclose
\pgfusepath{fill}
\end{pgfscope}
\begin{pgfscope}
\pgfsetdash{}{0cm}
\pgfsetlinewidth{0.818mm}
\pgfsetroundcap
\pgfsetroundjoin
\pgfsetmiterlimit{7.0}
\definecolor{eps2pgf_color}{gray}{0}\pgfsetstrokecolor{eps2pgf_color}\pgfsetfillcolor{eps2pgf_color}
\pgfpathmoveto{\pgfqpoint{0.117cm}{1.476cm}}
\pgfpathlineto{\pgfqpoint{0.682cm}{0.726cm}}
\pgfpathlineto{\pgfqpoint{1.246cm}{1.476cm}}
\pgfusepath{stroke}
\end{pgfscope}
\definecolor{eps2pgf_color}{gray}{0}\pgfsetstrokecolor{eps2pgf_color}\pgfsetfillcolor{eps2pgf_color}
\pgfpathmoveto{\pgfqpoint{0.273cm}{1.451cm}}
\pgfpathcurveto{\pgfqpoint{0.273cm}{1.487cm}}{\pgfqpoint{0.259cm}{1.522cm}}{\pgfqpoint{0.233cm}{1.547cm}}
\pgfpathcurveto{\pgfqpoint{0.207cm}{1.573cm}}{\pgfqpoint{0.173cm}{1.588cm}}{\pgfqpoint{0.137cm}{1.588cm}}
\pgfpathcurveto{\pgfqpoint{0.1cm}{1.588cm}}{\pgfqpoint{0.066cm}{1.573cm}}{\pgfqpoint{0.04cm}{1.547cm}}
\pgfpathcurveto{\pgfqpoint{0.014cm}{1.522cm}}{\pgfqpoint{0cm}{1.487cm}}{\pgfqpoint{0cm}{1.451cm}}
\pgfpathcurveto{\pgfqpoint{0cm}{1.414cm}}{\pgfqpoint{0.014cm}{1.379cm}}{\pgfqpoint{0.04cm}{1.354cm}}
\pgfpathcurveto{\pgfqpoint{0.066cm}{1.328cm}}{\pgfqpoint{0.1cm}{1.314cm}}{\pgfqpoint{0.137cm}{1.314cm}}
\pgfpathcurveto{\pgfqpoint{0.173cm}{1.314cm}}{\pgfqpoint{0.207cm}{1.328cm}}{\pgfqpoint{0.233cm}{1.354cm}}
\pgfpathcurveto{\pgfqpoint{0.259cm}{1.379cm}}{\pgfqpoint{0.273cm}{1.414cm}}{\pgfqpoint{0.273cm}{1.451cm}}
\pgfusepath{fill}
\pgfpathmoveto{\pgfqpoint{1.345cm}{1.426cm}}
\pgfpathcurveto{\pgfqpoint{1.345cm}{1.463cm}}{\pgfqpoint{1.331cm}{1.497cm}}{\pgfqpoint{1.305cm}{1.523cm}}
\pgfpathcurveto{\pgfqpoint{1.28cm}{1.549cm}}{\pgfqpoint{1.245cm}{1.563cm}}{\pgfqpoint{1.209cm}{1.563cm}}
\pgfpathcurveto{\pgfqpoint{1.172cm}{1.563cm}}{\pgfqpoint{1.138cm}{1.549cm}}{\pgfqpoint{1.112cm}{1.523cm}}
\pgfpathcurveto{\pgfqpoint{1.087cm}{1.497cm}}{\pgfqpoint{1.072cm}{1.463cm}}{\pgfqpoint{1.072cm}{1.426cm}}
\pgfpathcurveto{\pgfqpoint{1.072cm}{1.39cm}}{\pgfqpoint{1.087cm}{1.355cm}}{\pgfqpoint{1.112cm}{1.329cm}}
\pgfpathcurveto{\pgfqpoint{1.138cm}{1.304cm}}{\pgfqpoint{1.172cm}{1.289cm}}{\pgfqpoint{1.209cm}{1.289cm}}
\pgfpathcurveto{\pgfqpoint{1.245cm}{1.289cm}}{\pgfqpoint{1.28cm}{1.304cm}}{\pgfqpoint{1.305cm}{1.329cm}}
\pgfpathcurveto{\pgfqpoint{1.331cm}{1.355cm}}{\pgfqpoint{1.345cm}{1.39cm}}{\pgfqpoint{1.345cm}{1.426cm}}
\pgfusepath{fill}
\begin{pgfscope}
\pgfsetdash{}{0cm}
\pgfsetlinewidth{0.818mm}
\pgfsetroundcap
\pgfsetmiterlimit{4.0}
\pgfpathmoveto{\pgfqpoint{0.682cm}{0.726cm}}
\pgfpathlineto{\pgfqpoint{0.682cm}{0.097cm}}
\pgfusepath{stroke}
\end{pgfscope}
\end{pgfscope}
\end{pgfscope}
\end{pgfscope}
\end{tikzpicture}}}-(-X^{\!\resizebox{0.6em}{!}{
\begin{tikzpicture}
\pgfpathmoveto{\pgfqpoint{0cm}{-0.035cm}}
\pgfpathlineto{\pgfqpoint{1.376cm}{-0.035cm}}
\pgfpathlineto{\pgfqpoint{1.376cm}{1.552cm}}
\pgfpathlineto{\pgfqpoint{0cm}{1.552cm}}
\pgfpathclose
\pgfusepath{clip}
\begin{pgfscope}
\begin{pgfscope}
\pgfpathmoveto{\pgfqpoint{0cm}{-0.035cm}}
\pgfpathlineto{\pgfqpoint{1.376cm}{-0.035cm}}
\pgfpathlineto{\pgfqpoint{1.376cm}{1.552cm}}
\pgfpathlineto{\pgfqpoint{0cm}{1.552cm}}
\pgfpathclose
\pgfusepath{clip}
\begin{pgfscope}
\begin{pgfscope}
\pgfsetdash{}{0cm}
\pgfsetlinewidth{0.818mm}
\pgfsetroundcap
\pgfsetroundjoin
\pgfsetmiterlimit{7.0}
\definecolor{eps2pgf_color}{gray}{0}\pgfsetstrokecolor{eps2pgf_color}\pgfsetfillcolor{eps2pgf_color}
\pgfpathmoveto{\pgfqpoint{0.117cm}{1.421cm}}
\pgfpathlineto{\pgfqpoint{0.682cm}{0.671cm}}
\pgfpathlineto{\pgfqpoint{1.246cm}{1.421cm}}
\pgfusepath{stroke}
\end{pgfscope}
\definecolor{eps2pgf_color}{gray}{0}\pgfsetstrokecolor{eps2pgf_color}\pgfsetfillcolor{eps2pgf_color}
\pgfpathmoveto{\pgfqpoint{0.273cm}{1.395cm}}
\pgfpathcurveto{\pgfqpoint{0.273cm}{1.432cm}}{\pgfqpoint{0.259cm}{1.467cm}}{\pgfqpoint{0.233cm}{1.492cm}}
\pgfpathcurveto{\pgfqpoint{0.207cm}{1.518cm}}{\pgfqpoint{0.173cm}{1.532cm}}{\pgfqpoint{0.137cm}{1.532cm}}
\pgfpathcurveto{\pgfqpoint{0.1cm}{1.532cm}}{\pgfqpoint{0.066cm}{1.518cm}}{\pgfqpoint{0.04cm}{1.492cm}}
\pgfpathcurveto{\pgfqpoint{0.014cm}{1.467cm}}{\pgfqpoint{0cm}{1.432cm}}{\pgfqpoint{0cm}{1.395cm}}
\pgfpathcurveto{\pgfqpoint{0cm}{1.359cm}}{\pgfqpoint{0.014cm}{1.324cm}}{\pgfqpoint{0.04cm}{1.299cm}}
\pgfpathcurveto{\pgfqpoint{0.066cm}{1.273cm}}{\pgfqpoint{0.1cm}{1.258cm}}{\pgfqpoint{0.137cm}{1.258cm}}
\pgfpathcurveto{\pgfqpoint{0.173cm}{1.258cm}}{\pgfqpoint{0.207cm}{1.273cm}}{\pgfqpoint{0.233cm}{1.299cm}}
\pgfpathcurveto{\pgfqpoint{0.259cm}{1.324cm}}{\pgfqpoint{0.273cm}{1.359cm}}{\pgfqpoint{0.273cm}{1.395cm}}
\pgfusepath{fill}
\begin{pgfscope}
\pgfsetdash{}{0cm}
\pgfsetlinewidth{0.818mm}
\pgfsetmiterlimit{7.0}
\pgfpathmoveto{\pgfqpoint{0.682cm}{0.671cm}}
\pgfpathlineto{\pgfqpoint{0.679cm}{1.418cm}}
\pgfusepath{stroke}
\end{pgfscope}
\pgfpathmoveto{\pgfqpoint{0.815cm}{1.399cm}}
\pgfpathcurveto{\pgfqpoint{0.815cm}{1.435cm}}{\pgfqpoint{0.801cm}{1.47cm}}{\pgfqpoint{0.775cm}{1.496cm}}
\pgfpathcurveto{\pgfqpoint{0.75cm}{1.521cm}}{\pgfqpoint{0.715cm}{1.536cm}}{\pgfqpoint{0.679cm}{1.536cm}}
\pgfpathcurveto{\pgfqpoint{0.643cm}{1.536cm}}{\pgfqpoint{0.608cm}{1.521cm}}{\pgfqpoint{0.582cm}{1.496cm}}
\pgfpathcurveto{\pgfqpoint{0.557cm}{1.47cm}}{\pgfqpoint{0.542cm}{1.435cm}}{\pgfqpoint{0.542cm}{1.399cm}}
\pgfpathcurveto{\pgfqpoint{0.542cm}{1.363cm}}{\pgfqpoint{0.557cm}{1.328cm}}{\pgfqpoint{0.582cm}{1.302cm}}
\pgfpathcurveto{\pgfqpoint{0.608cm}{1.276cm}}{\pgfqpoint{0.643cm}{1.262cm}}{\pgfqpoint{0.679cm}{1.262cm}}
\pgfpathcurveto{\pgfqpoint{0.715cm}{1.262cm}}{\pgfqpoint{0.75cm}{1.276cm}}{\pgfqpoint{0.775cm}{1.302cm}}
\pgfpathcurveto{\pgfqpoint{0.801cm}{1.328cm}}{\pgfqpoint{0.815cm}{1.363cm}}{\pgfqpoint{0.815cm}{1.399cm}}
\pgfusepath{fill}
\pgfpathmoveto{\pgfqpoint{1.345cm}{1.371cm}}
\pgfpathcurveto{\pgfqpoint{1.345cm}{1.408cm}}{\pgfqpoint{1.331cm}{1.442cm}}{\pgfqpoint{1.305cm}{1.468cm}}
\pgfpathcurveto{\pgfqpoint{1.28cm}{1.494cm}}{\pgfqpoint{1.245cm}{1.508cm}}{\pgfqpoint{1.209cm}{1.508cm}}
\pgfpathcurveto{\pgfqpoint{1.172cm}{1.508cm}}{\pgfqpoint{1.138cm}{1.494cm}}{\pgfqpoint{1.112cm}{1.468cm}}
\pgfpathcurveto{\pgfqpoint{1.087cm}{1.442cm}}{\pgfqpoint{1.072cm}{1.408cm}}{\pgfqpoint{1.072cm}{1.371cm}}
\pgfpathcurveto{\pgfqpoint{1.072cm}{1.335cm}}{\pgfqpoint{1.087cm}{1.3cm}}{\pgfqpoint{1.112cm}{1.274cm}}
\pgfpathcurveto{\pgfqpoint{1.138cm}{1.249cm}}{\pgfqpoint{1.172cm}{1.234cm}}{\pgfqpoint{1.209cm}{1.234cm}}
\pgfpathcurveto{\pgfqpoint{1.245cm}{1.234cm}}{\pgfqpoint{1.28cm}{1.249cm}}{\pgfqpoint{1.305cm}{1.274cm}}
\pgfpathcurveto{\pgfqpoint{1.331cm}{1.3cm}}{\pgfqpoint{1.345cm}{1.335cm}}{\pgfqpoint{1.345cm}{1.371cm}}
\pgfusepath{fill}
\begin{pgfscope}
\pgfsetdash{}{0cm}
\pgfsetlinewidth{0.818mm}
\pgfsetroundcap
\pgfsetmiterlimit{4.0}
\pgfpathmoveto{\pgfqpoint{0.682cm}{0.671cm}}
\pgfpathlineto{\pgfqpoint{0.682cm}{0.042cm}}
\pgfusepath{stroke}
\end{pgfscope}
\end{pgfscope}
\end{pgfscope}
\end{pgfscope}
\end{tikzpicture}}}+\phi+\psi)\prec X^{\!\resizebox{0.6em}{!}{
\begin{tikzpicture}
\pgfpathmoveto{\pgfqpoint{0cm}{0cm}}
\pgfpathlineto{\pgfqpoint{1.376cm}{0cm}}
\pgfpathlineto{\pgfqpoint{1.376cm}{1.588cm}}
\pgfpathlineto{\pgfqpoint{0cm}{1.588cm}}
\pgfpathclose
\pgfusepath{clip}
\begin{pgfscope}
\begin{pgfscope}
\pgfpathmoveto{\pgfqpoint{0cm}{0cm}}
\pgfpathlineto{\pgfqpoint{1.376cm}{0cm}}
\pgfpathlineto{\pgfqpoint{1.376cm}{1.588cm}}
\pgfpathlineto{\pgfqpoint{0cm}{1.588cm}}
\pgfpathclose
\pgfusepath{clip}
\begin{pgfscope}
\begin{pgfscope}
\definecolor{eps2pgf_color}{gray}{0.976471}\pgfsetstrokecolor{eps2pgf_color}\pgfsetfillcolor{eps2pgf_color}
\pgfpathmoveto{\pgfqpoint{0cm}{0cm}}
\pgfpathlineto{\pgfqpoint{1.376cm}{0cm}}
\pgfpathlineto{\pgfqpoint{1.376cm}{1.588cm}}
\pgfpathlineto{\pgfqpoint{0cm}{1.588cm}}
\pgfpathclose
\pgfusepath{fill}
\end{pgfscope}
\begin{pgfscope}
\pgfsetdash{}{0cm}
\pgfsetlinewidth{0.818mm}
\pgfsetroundcap
\pgfsetroundjoin
\pgfsetmiterlimit{7.0}
\definecolor{eps2pgf_color}{gray}{0}\pgfsetstrokecolor{eps2pgf_color}\pgfsetfillcolor{eps2pgf_color}
\pgfpathmoveto{\pgfqpoint{0.117cm}{1.476cm}}
\pgfpathlineto{\pgfqpoint{0.682cm}{0.726cm}}
\pgfpathlineto{\pgfqpoint{1.246cm}{1.476cm}}
\pgfusepath{stroke}
\end{pgfscope}
\definecolor{eps2pgf_color}{gray}{0}\pgfsetstrokecolor{eps2pgf_color}\pgfsetfillcolor{eps2pgf_color}
\pgfpathmoveto{\pgfqpoint{0.273cm}{1.451cm}}
\pgfpathcurveto{\pgfqpoint{0.273cm}{1.487cm}}{\pgfqpoint{0.259cm}{1.522cm}}{\pgfqpoint{0.233cm}{1.547cm}}
\pgfpathcurveto{\pgfqpoint{0.207cm}{1.573cm}}{\pgfqpoint{0.173cm}{1.588cm}}{\pgfqpoint{0.137cm}{1.588cm}}
\pgfpathcurveto{\pgfqpoint{0.1cm}{1.588cm}}{\pgfqpoint{0.066cm}{1.573cm}}{\pgfqpoint{0.04cm}{1.547cm}}
\pgfpathcurveto{\pgfqpoint{0.014cm}{1.522cm}}{\pgfqpoint{0cm}{1.487cm}}{\pgfqpoint{0cm}{1.451cm}}
\pgfpathcurveto{\pgfqpoint{0cm}{1.414cm}}{\pgfqpoint{0.014cm}{1.379cm}}{\pgfqpoint{0.04cm}{1.354cm}}
\pgfpathcurveto{\pgfqpoint{0.066cm}{1.328cm}}{\pgfqpoint{0.1cm}{1.314cm}}{\pgfqpoint{0.137cm}{1.314cm}}
\pgfpathcurveto{\pgfqpoint{0.173cm}{1.314cm}}{\pgfqpoint{0.207cm}{1.328cm}}{\pgfqpoint{0.233cm}{1.354cm}}
\pgfpathcurveto{\pgfqpoint{0.259cm}{1.379cm}}{\pgfqpoint{0.273cm}{1.414cm}}{\pgfqpoint{0.273cm}{1.451cm}}
\pgfusepath{fill}
\pgfpathmoveto{\pgfqpoint{1.345cm}{1.426cm}}
\pgfpathcurveto{\pgfqpoint{1.345cm}{1.463cm}}{\pgfqpoint{1.331cm}{1.497cm}}{\pgfqpoint{1.305cm}{1.523cm}}
\pgfpathcurveto{\pgfqpoint{1.28cm}{1.549cm}}{\pgfqpoint{1.245cm}{1.563cm}}{\pgfqpoint{1.209cm}{1.563cm}}
\pgfpathcurveto{\pgfqpoint{1.172cm}{1.563cm}}{\pgfqpoint{1.138cm}{1.549cm}}{\pgfqpoint{1.112cm}{1.523cm}}
\pgfpathcurveto{\pgfqpoint{1.087cm}{1.497cm}}{\pgfqpoint{1.072cm}{1.463cm}}{\pgfqpoint{1.072cm}{1.426cm}}
\pgfpathcurveto{\pgfqpoint{1.072cm}{1.39cm}}{\pgfqpoint{1.087cm}{1.355cm}}{\pgfqpoint{1.112cm}{1.329cm}}
\pgfpathcurveto{\pgfqpoint{1.138cm}{1.304cm}}{\pgfqpoint{1.172cm}{1.289cm}}{\pgfqpoint{1.209cm}{1.289cm}}
\pgfpathcurveto{\pgfqpoint{1.245cm}{1.289cm}}{\pgfqpoint{1.28cm}{1.304cm}}{\pgfqpoint{1.305cm}{1.329cm}}
\pgfpathcurveto{\pgfqpoint{1.331cm}{1.355cm}}{\pgfqpoint{1.345cm}{1.39cm}}{\pgfqpoint{1.345cm}{1.426cm}}
\pgfusepath{fill}
\begin{pgfscope}
\pgfsetdash{}{0cm}
\pgfsetlinewidth{0.818mm}
\pgfsetroundcap
\pgfsetmiterlimit{4.0}
\pgfpathmoveto{\pgfqpoint{0.682cm}{0.726cm}}
\pgfpathlineto{\pgfqpoint{0.682cm}{0.097cm}}
\pgfusepath{stroke}
\end{pgfscope}
\end{pgfscope}
\end{pgfscope}
\end{pgfscope}
\end{tikzpicture}}}],
 \end{equation}
 and $\rmm{\Phi}$, $\rmb{\Psi}$ are given exactly as in Section \ref{sec:45} (using the parabolic localizers $\VV_{\leq},\VV_{>}$ instead of the elliptic ones $\UU_{\leq},\UU_{>}$) with the same bounds applied pointwise in time, and $\varphi_{0}\in \CC^{1+\alpha}(\rho^{3+\gamma'}_{0})$ for some $\gamma'\in(0,\gamma)$ to be chosen below (the role of the parameter $\gamma$ is the same as in Section \ref{sec:45}, in particular, $\gamma=\alpha-\kappa$). The additional commutator in \eqref{eq:two43aa} is bounded as
 \begin{align*}
&\| 9\llbracket X^{2}\rrbracket \circ[(-X^{\!\resizebox{0.6em}{!}{
\begin{tikzpicture}
\pgfpathmoveto{\pgfqpoint{0cm}{-0.035cm}}
\pgfpathlineto{\pgfqpoint{1.376cm}{-0.035cm}}
\pgfpathlineto{\pgfqpoint{1.376cm}{1.552cm}}
\pgfpathlineto{\pgfqpoint{0cm}{1.552cm}}
\pgfpathclose
\pgfusepath{clip}
\begin{pgfscope}
\begin{pgfscope}
\pgfpathmoveto{\pgfqpoint{0cm}{-0.035cm}}
\pgfpathlineto{\pgfqpoint{1.376cm}{-0.035cm}}
\pgfpathlineto{\pgfqpoint{1.376cm}{1.552cm}}
\pgfpathlineto{\pgfqpoint{0cm}{1.552cm}}
\pgfpathclose
\pgfusepath{clip}
\begin{pgfscope}
\begin{pgfscope}
\pgfsetdash{}{0cm}
\pgfsetlinewidth{0.818mm}
\pgfsetroundcap
\pgfsetroundjoin
\pgfsetmiterlimit{7.0}
\definecolor{eps2pgf_color}{gray}{0}\pgfsetstrokecolor{eps2pgf_color}\pgfsetfillcolor{eps2pgf_color}
\pgfpathmoveto{\pgfqpoint{0.117cm}{1.421cm}}
\pgfpathlineto{\pgfqpoint{0.682cm}{0.671cm}}
\pgfpathlineto{\pgfqpoint{1.246cm}{1.421cm}}
\pgfusepath{stroke}
\end{pgfscope}
\definecolor{eps2pgf_color}{gray}{0}\pgfsetstrokecolor{eps2pgf_color}\pgfsetfillcolor{eps2pgf_color}
\pgfpathmoveto{\pgfqpoint{0.273cm}{1.395cm}}
\pgfpathcurveto{\pgfqpoint{0.273cm}{1.432cm}}{\pgfqpoint{0.259cm}{1.467cm}}{\pgfqpoint{0.233cm}{1.492cm}}
\pgfpathcurveto{\pgfqpoint{0.207cm}{1.518cm}}{\pgfqpoint{0.173cm}{1.532cm}}{\pgfqpoint{0.137cm}{1.532cm}}
\pgfpathcurveto{\pgfqpoint{0.1cm}{1.532cm}}{\pgfqpoint{0.066cm}{1.518cm}}{\pgfqpoint{0.04cm}{1.492cm}}
\pgfpathcurveto{\pgfqpoint{0.014cm}{1.467cm}}{\pgfqpoint{0cm}{1.432cm}}{\pgfqpoint{0cm}{1.395cm}}
\pgfpathcurveto{\pgfqpoint{0cm}{1.359cm}}{\pgfqpoint{0.014cm}{1.324cm}}{\pgfqpoint{0.04cm}{1.299cm}}
\pgfpathcurveto{\pgfqpoint{0.066cm}{1.273cm}}{\pgfqpoint{0.1cm}{1.258cm}}{\pgfqpoint{0.137cm}{1.258cm}}
\pgfpathcurveto{\pgfqpoint{0.173cm}{1.258cm}}{\pgfqpoint{0.207cm}{1.273cm}}{\pgfqpoint{0.233cm}{1.299cm}}
\pgfpathcurveto{\pgfqpoint{0.259cm}{1.324cm}}{\pgfqpoint{0.273cm}{1.359cm}}{\pgfqpoint{0.273cm}{1.395cm}}
\pgfusepath{fill}
\begin{pgfscope}
\pgfsetdash{}{0cm}
\pgfsetlinewidth{0.818mm}
\pgfsetmiterlimit{7.0}
\pgfpathmoveto{\pgfqpoint{0.682cm}{0.671cm}}
\pgfpathlineto{\pgfqpoint{0.679cm}{1.418cm}}
\pgfusepath{stroke}
\end{pgfscope}
\pgfpathmoveto{\pgfqpoint{0.815cm}{1.399cm}}
\pgfpathcurveto{\pgfqpoint{0.815cm}{1.435cm}}{\pgfqpoint{0.801cm}{1.47cm}}{\pgfqpoint{0.775cm}{1.496cm}}
\pgfpathcurveto{\pgfqpoint{0.75cm}{1.521cm}}{\pgfqpoint{0.715cm}{1.536cm}}{\pgfqpoint{0.679cm}{1.536cm}}
\pgfpathcurveto{\pgfqpoint{0.643cm}{1.536cm}}{\pgfqpoint{0.608cm}{1.521cm}}{\pgfqpoint{0.582cm}{1.496cm}}
\pgfpathcurveto{\pgfqpoint{0.557cm}{1.47cm}}{\pgfqpoint{0.542cm}{1.435cm}}{\pgfqpoint{0.542cm}{1.399cm}}
\pgfpathcurveto{\pgfqpoint{0.542cm}{1.363cm}}{\pgfqpoint{0.557cm}{1.328cm}}{\pgfqpoint{0.582cm}{1.302cm}}
\pgfpathcurveto{\pgfqpoint{0.608cm}{1.276cm}}{\pgfqpoint{0.643cm}{1.262cm}}{\pgfqpoint{0.679cm}{1.262cm}}
\pgfpathcurveto{\pgfqpoint{0.715cm}{1.262cm}}{\pgfqpoint{0.75cm}{1.276cm}}{\pgfqpoint{0.775cm}{1.302cm}}
\pgfpathcurveto{\pgfqpoint{0.801cm}{1.328cm}}{\pgfqpoint{0.815cm}{1.363cm}}{\pgfqpoint{0.815cm}{1.399cm}}
\pgfusepath{fill}
\pgfpathmoveto{\pgfqpoint{1.345cm}{1.371cm}}
\pgfpathcurveto{\pgfqpoint{1.345cm}{1.408cm}}{\pgfqpoint{1.331cm}{1.442cm}}{\pgfqpoint{1.305cm}{1.468cm}}
\pgfpathcurveto{\pgfqpoint{1.28cm}{1.494cm}}{\pgfqpoint{1.245cm}{1.508cm}}{\pgfqpoint{1.209cm}{1.508cm}}
\pgfpathcurveto{\pgfqpoint{1.172cm}{1.508cm}}{\pgfqpoint{1.138cm}{1.494cm}}{\pgfqpoint{1.112cm}{1.468cm}}
\pgfpathcurveto{\pgfqpoint{1.087cm}{1.442cm}}{\pgfqpoint{1.072cm}{1.408cm}}{\pgfqpoint{1.072cm}{1.371cm}}
\pgfpathcurveto{\pgfqpoint{1.072cm}{1.335cm}}{\pgfqpoint{1.087cm}{1.3cm}}{\pgfqpoint{1.112cm}{1.274cm}}
\pgfpathcurveto{\pgfqpoint{1.138cm}{1.249cm}}{\pgfqpoint{1.172cm}{1.234cm}}{\pgfqpoint{1.209cm}{1.234cm}}
\pgfpathcurveto{\pgfqpoint{1.245cm}{1.234cm}}{\pgfqpoint{1.28cm}{1.249cm}}{\pgfqpoint{1.305cm}{1.274cm}}
\pgfpathcurveto{\pgfqpoint{1.331cm}{1.3cm}}{\pgfqpoint{1.345cm}{1.335cm}}{\pgfqpoint{1.345cm}{1.371cm}}
\pgfusepath{fill}
\begin{pgfscope}
\pgfsetdash{}{0cm}
\pgfsetlinewidth{0.818mm}
\pgfsetroundcap
\pgfsetmiterlimit{4.0}
\pgfpathmoveto{\pgfqpoint{0.682cm}{0.671cm}}
\pgfpathlineto{\pgfqpoint{0.682cm}{0.042cm}}
\pgfusepath{stroke}
\end{pgfscope}
\end{pgfscope}
\end{pgfscope}
\end{pgfscope}
\end{tikzpicture}}}+\phi+\psi)\precprec X^{\!\resizebox{0.6em}{!}{
\begin{tikzpicture}
\pgfpathmoveto{\pgfqpoint{0cm}{0cm}}
\pgfpathlineto{\pgfqpoint{1.376cm}{0cm}}
\pgfpathlineto{\pgfqpoint{1.376cm}{1.588cm}}
\pgfpathlineto{\pgfqpoint{0cm}{1.588cm}}
\pgfpathclose
\pgfusepath{clip}
\begin{pgfscope}
\begin{pgfscope}
\pgfpathmoveto{\pgfqpoint{0cm}{0cm}}
\pgfpathlineto{\pgfqpoint{1.376cm}{0cm}}
\pgfpathlineto{\pgfqpoint{1.376cm}{1.588cm}}
\pgfpathlineto{\pgfqpoint{0cm}{1.588cm}}
\pgfpathclose
\pgfusepath{clip}
\begin{pgfscope}
\begin{pgfscope}
\definecolor{eps2pgf_color}{gray}{0.976471}\pgfsetstrokecolor{eps2pgf_color}\pgfsetfillcolor{eps2pgf_color}
\pgfpathmoveto{\pgfqpoint{0cm}{0cm}}
\pgfpathlineto{\pgfqpoint{1.376cm}{0cm}}
\pgfpathlineto{\pgfqpoint{1.376cm}{1.588cm}}
\pgfpathlineto{\pgfqpoint{0cm}{1.588cm}}
\pgfpathclose
\pgfusepath{fill}
\end{pgfscope}
\begin{pgfscope}
\pgfsetdash{}{0cm}
\pgfsetlinewidth{0.818mm}
\pgfsetroundcap
\pgfsetroundjoin
\pgfsetmiterlimit{7.0}
\definecolor{eps2pgf_color}{gray}{0}\pgfsetstrokecolor{eps2pgf_color}\pgfsetfillcolor{eps2pgf_color}
\pgfpathmoveto{\pgfqpoint{0.117cm}{1.476cm}}
\pgfpathlineto{\pgfqpoint{0.682cm}{0.726cm}}
\pgfpathlineto{\pgfqpoint{1.246cm}{1.476cm}}
\pgfusepath{stroke}
\end{pgfscope}
\definecolor{eps2pgf_color}{gray}{0}\pgfsetstrokecolor{eps2pgf_color}\pgfsetfillcolor{eps2pgf_color}
\pgfpathmoveto{\pgfqpoint{0.273cm}{1.451cm}}
\pgfpathcurveto{\pgfqpoint{0.273cm}{1.487cm}}{\pgfqpoint{0.259cm}{1.522cm}}{\pgfqpoint{0.233cm}{1.547cm}}
\pgfpathcurveto{\pgfqpoint{0.207cm}{1.573cm}}{\pgfqpoint{0.173cm}{1.588cm}}{\pgfqpoint{0.137cm}{1.588cm}}
\pgfpathcurveto{\pgfqpoint{0.1cm}{1.588cm}}{\pgfqpoint{0.066cm}{1.573cm}}{\pgfqpoint{0.04cm}{1.547cm}}
\pgfpathcurveto{\pgfqpoint{0.014cm}{1.522cm}}{\pgfqpoint{0cm}{1.487cm}}{\pgfqpoint{0cm}{1.451cm}}
\pgfpathcurveto{\pgfqpoint{0cm}{1.414cm}}{\pgfqpoint{0.014cm}{1.379cm}}{\pgfqpoint{0.04cm}{1.354cm}}
\pgfpathcurveto{\pgfqpoint{0.066cm}{1.328cm}}{\pgfqpoint{0.1cm}{1.314cm}}{\pgfqpoint{0.137cm}{1.314cm}}
\pgfpathcurveto{\pgfqpoint{0.173cm}{1.314cm}}{\pgfqpoint{0.207cm}{1.328cm}}{\pgfqpoint{0.233cm}{1.354cm}}
\pgfpathcurveto{\pgfqpoint{0.259cm}{1.379cm}}{\pgfqpoint{0.273cm}{1.414cm}}{\pgfqpoint{0.273cm}{1.451cm}}
\pgfusepath{fill}
\pgfpathmoveto{\pgfqpoint{1.345cm}{1.426cm}}
\pgfpathcurveto{\pgfqpoint{1.345cm}{1.463cm}}{\pgfqpoint{1.331cm}{1.497cm}}{\pgfqpoint{1.305cm}{1.523cm}}
\pgfpathcurveto{\pgfqpoint{1.28cm}{1.549cm}}{\pgfqpoint{1.245cm}{1.563cm}}{\pgfqpoint{1.209cm}{1.563cm}}
\pgfpathcurveto{\pgfqpoint{1.172cm}{1.563cm}}{\pgfqpoint{1.138cm}{1.549cm}}{\pgfqpoint{1.112cm}{1.523cm}}
\pgfpathcurveto{\pgfqpoint{1.087cm}{1.497cm}}{\pgfqpoint{1.072cm}{1.463cm}}{\pgfqpoint{1.072cm}{1.426cm}}
\pgfpathcurveto{\pgfqpoint{1.072cm}{1.39cm}}{\pgfqpoint{1.087cm}{1.355cm}}{\pgfqpoint{1.112cm}{1.329cm}}
\pgfpathcurveto{\pgfqpoint{1.138cm}{1.304cm}}{\pgfqpoint{1.172cm}{1.289cm}}{\pgfqpoint{1.209cm}{1.289cm}}
\pgfpathcurveto{\pgfqpoint{1.245cm}{1.289cm}}{\pgfqpoint{1.28cm}{1.304cm}}{\pgfqpoint{1.305cm}{1.329cm}}
\pgfpathcurveto{\pgfqpoint{1.331cm}{1.355cm}}{\pgfqpoint{1.345cm}{1.39cm}}{\pgfqpoint{1.345cm}{1.426cm}}
\pgfusepath{fill}
\begin{pgfscope}
\pgfsetdash{}{0cm}
\pgfsetlinewidth{0.818mm}
\pgfsetroundcap
\pgfsetmiterlimit{4.0}
\pgfpathmoveto{\pgfqpoint{0.682cm}{0.726cm}}
\pgfpathlineto{\pgfqpoint{0.682cm}{0.097cm}}
\pgfusepath{stroke}
\end{pgfscope}
\end{pgfscope}
\end{pgfscope}
\end{pgfscope}
\end{tikzpicture}}}-(-X^{\!\resizebox{0.6em}{!}{
\begin{tikzpicture}
\pgfpathmoveto{\pgfqpoint{0cm}{-0.035cm}}
\pgfpathlineto{\pgfqpoint{1.376cm}{-0.035cm}}
\pgfpathlineto{\pgfqpoint{1.376cm}{1.552cm}}
\pgfpathlineto{\pgfqpoint{0cm}{1.552cm}}
\pgfpathclose
\pgfusepath{clip}
\begin{pgfscope}
\begin{pgfscope}
\pgfpathmoveto{\pgfqpoint{0cm}{-0.035cm}}
\pgfpathlineto{\pgfqpoint{1.376cm}{-0.035cm}}
\pgfpathlineto{\pgfqpoint{1.376cm}{1.552cm}}
\pgfpathlineto{\pgfqpoint{0cm}{1.552cm}}
\pgfpathclose
\pgfusepath{clip}
\begin{pgfscope}
\begin{pgfscope}
\pgfsetdash{}{0cm}
\pgfsetlinewidth{0.818mm}
\pgfsetroundcap
\pgfsetroundjoin
\pgfsetmiterlimit{7.0}
\definecolor{eps2pgf_color}{gray}{0}\pgfsetstrokecolor{eps2pgf_color}\pgfsetfillcolor{eps2pgf_color}
\pgfpathmoveto{\pgfqpoint{0.117cm}{1.421cm}}
\pgfpathlineto{\pgfqpoint{0.682cm}{0.671cm}}
\pgfpathlineto{\pgfqpoint{1.246cm}{1.421cm}}
\pgfusepath{stroke}
\end{pgfscope}
\definecolor{eps2pgf_color}{gray}{0}\pgfsetstrokecolor{eps2pgf_color}\pgfsetfillcolor{eps2pgf_color}
\pgfpathmoveto{\pgfqpoint{0.273cm}{1.395cm}}
\pgfpathcurveto{\pgfqpoint{0.273cm}{1.432cm}}{\pgfqpoint{0.259cm}{1.467cm}}{\pgfqpoint{0.233cm}{1.492cm}}
\pgfpathcurveto{\pgfqpoint{0.207cm}{1.518cm}}{\pgfqpoint{0.173cm}{1.532cm}}{\pgfqpoint{0.137cm}{1.532cm}}
\pgfpathcurveto{\pgfqpoint{0.1cm}{1.532cm}}{\pgfqpoint{0.066cm}{1.518cm}}{\pgfqpoint{0.04cm}{1.492cm}}
\pgfpathcurveto{\pgfqpoint{0.014cm}{1.467cm}}{\pgfqpoint{0cm}{1.432cm}}{\pgfqpoint{0cm}{1.395cm}}
\pgfpathcurveto{\pgfqpoint{0cm}{1.359cm}}{\pgfqpoint{0.014cm}{1.324cm}}{\pgfqpoint{0.04cm}{1.299cm}}
\pgfpathcurveto{\pgfqpoint{0.066cm}{1.273cm}}{\pgfqpoint{0.1cm}{1.258cm}}{\pgfqpoint{0.137cm}{1.258cm}}
\pgfpathcurveto{\pgfqpoint{0.173cm}{1.258cm}}{\pgfqpoint{0.207cm}{1.273cm}}{\pgfqpoint{0.233cm}{1.299cm}}
\pgfpathcurveto{\pgfqpoint{0.259cm}{1.324cm}}{\pgfqpoint{0.273cm}{1.359cm}}{\pgfqpoint{0.273cm}{1.395cm}}
\pgfusepath{fill}
\begin{pgfscope}
\pgfsetdash{}{0cm}
\pgfsetlinewidth{0.818mm}
\pgfsetmiterlimit{7.0}
\pgfpathmoveto{\pgfqpoint{0.682cm}{0.671cm}}
\pgfpathlineto{\pgfqpoint{0.679cm}{1.418cm}}
\pgfusepath{stroke}
\end{pgfscope}
\pgfpathmoveto{\pgfqpoint{0.815cm}{1.399cm}}
\pgfpathcurveto{\pgfqpoint{0.815cm}{1.435cm}}{\pgfqpoint{0.801cm}{1.47cm}}{\pgfqpoint{0.775cm}{1.496cm}}
\pgfpathcurveto{\pgfqpoint{0.75cm}{1.521cm}}{\pgfqpoint{0.715cm}{1.536cm}}{\pgfqpoint{0.679cm}{1.536cm}}
\pgfpathcurveto{\pgfqpoint{0.643cm}{1.536cm}}{\pgfqpoint{0.608cm}{1.521cm}}{\pgfqpoint{0.582cm}{1.496cm}}
\pgfpathcurveto{\pgfqpoint{0.557cm}{1.47cm}}{\pgfqpoint{0.542cm}{1.435cm}}{\pgfqpoint{0.542cm}{1.399cm}}
\pgfpathcurveto{\pgfqpoint{0.542cm}{1.363cm}}{\pgfqpoint{0.557cm}{1.328cm}}{\pgfqpoint{0.582cm}{1.302cm}}
\pgfpathcurveto{\pgfqpoint{0.608cm}{1.276cm}}{\pgfqpoint{0.643cm}{1.262cm}}{\pgfqpoint{0.679cm}{1.262cm}}
\pgfpathcurveto{\pgfqpoint{0.715cm}{1.262cm}}{\pgfqpoint{0.75cm}{1.276cm}}{\pgfqpoint{0.775cm}{1.302cm}}
\pgfpathcurveto{\pgfqpoint{0.801cm}{1.328cm}}{\pgfqpoint{0.815cm}{1.363cm}}{\pgfqpoint{0.815cm}{1.399cm}}
\pgfusepath{fill}
\pgfpathmoveto{\pgfqpoint{1.345cm}{1.371cm}}
\pgfpathcurveto{\pgfqpoint{1.345cm}{1.408cm}}{\pgfqpoint{1.331cm}{1.442cm}}{\pgfqpoint{1.305cm}{1.468cm}}
\pgfpathcurveto{\pgfqpoint{1.28cm}{1.494cm}}{\pgfqpoint{1.245cm}{1.508cm}}{\pgfqpoint{1.209cm}{1.508cm}}
\pgfpathcurveto{\pgfqpoint{1.172cm}{1.508cm}}{\pgfqpoint{1.138cm}{1.494cm}}{\pgfqpoint{1.112cm}{1.468cm}}
\pgfpathcurveto{\pgfqpoint{1.087cm}{1.442cm}}{\pgfqpoint{1.072cm}{1.408cm}}{\pgfqpoint{1.072cm}{1.371cm}}
\pgfpathcurveto{\pgfqpoint{1.072cm}{1.335cm}}{\pgfqpoint{1.087cm}{1.3cm}}{\pgfqpoint{1.112cm}{1.274cm}}
\pgfpathcurveto{\pgfqpoint{1.138cm}{1.249cm}}{\pgfqpoint{1.172cm}{1.234cm}}{\pgfqpoint{1.209cm}{1.234cm}}
\pgfpathcurveto{\pgfqpoint{1.245cm}{1.234cm}}{\pgfqpoint{1.28cm}{1.249cm}}{\pgfqpoint{1.305cm}{1.274cm}}
\pgfpathcurveto{\pgfqpoint{1.331cm}{1.3cm}}{\pgfqpoint{1.345cm}{1.335cm}}{\pgfqpoint{1.345cm}{1.371cm}}
\pgfusepath{fill}
\begin{pgfscope}
\pgfsetdash{}{0cm}
\pgfsetlinewidth{0.818mm}
\pgfsetroundcap
\pgfsetmiterlimit{4.0}
\pgfpathmoveto{\pgfqpoint{0.682cm}{0.671cm}}
\pgfpathlineto{\pgfqpoint{0.682cm}{0.042cm}}
\pgfusepath{stroke}
\end{pgfscope}
\end{pgfscope}
\end{pgfscope}
\end{pgfscope}
\end{tikzpicture}}}+\phi+\psi)\prec X^{\!\resizebox{0.6em}{!}{
\begin{tikzpicture}
\pgfpathmoveto{\pgfqpoint{0cm}{0cm}}
\pgfpathlineto{\pgfqpoint{1.376cm}{0cm}}
\pgfpathlineto{\pgfqpoint{1.376cm}{1.588cm}}
\pgfpathlineto{\pgfqpoint{0cm}{1.588cm}}
\pgfpathclose
\pgfusepath{clip}
\begin{pgfscope}
\begin{pgfscope}
\pgfpathmoveto{\pgfqpoint{0cm}{0cm}}
\pgfpathlineto{\pgfqpoint{1.376cm}{0cm}}
\pgfpathlineto{\pgfqpoint{1.376cm}{1.588cm}}
\pgfpathlineto{\pgfqpoint{0cm}{1.588cm}}
\pgfpathclose
\pgfusepath{clip}
\begin{pgfscope}
\begin{pgfscope}
\definecolor{eps2pgf_color}{gray}{0.976471}\pgfsetstrokecolor{eps2pgf_color}\pgfsetfillcolor{eps2pgf_color}
\pgfpathmoveto{\pgfqpoint{0cm}{0cm}}
\pgfpathlineto{\pgfqpoint{1.376cm}{0cm}}
\pgfpathlineto{\pgfqpoint{1.376cm}{1.588cm}}
\pgfpathlineto{\pgfqpoint{0cm}{1.588cm}}
\pgfpathclose
\pgfusepath{fill}
\end{pgfscope}
\begin{pgfscope}
\pgfsetdash{}{0cm}
\pgfsetlinewidth{0.818mm}
\pgfsetroundcap
\pgfsetroundjoin
\pgfsetmiterlimit{7.0}
\definecolor{eps2pgf_color}{gray}{0}\pgfsetstrokecolor{eps2pgf_color}\pgfsetfillcolor{eps2pgf_color}
\pgfpathmoveto{\pgfqpoint{0.117cm}{1.476cm}}
\pgfpathlineto{\pgfqpoint{0.682cm}{0.726cm}}
\pgfpathlineto{\pgfqpoint{1.246cm}{1.476cm}}
\pgfusepath{stroke}
\end{pgfscope}
\definecolor{eps2pgf_color}{gray}{0}\pgfsetstrokecolor{eps2pgf_color}\pgfsetfillcolor{eps2pgf_color}
\pgfpathmoveto{\pgfqpoint{0.273cm}{1.451cm}}
\pgfpathcurveto{\pgfqpoint{0.273cm}{1.487cm}}{\pgfqpoint{0.259cm}{1.522cm}}{\pgfqpoint{0.233cm}{1.547cm}}
\pgfpathcurveto{\pgfqpoint{0.207cm}{1.573cm}}{\pgfqpoint{0.173cm}{1.588cm}}{\pgfqpoint{0.137cm}{1.588cm}}
\pgfpathcurveto{\pgfqpoint{0.1cm}{1.588cm}}{\pgfqpoint{0.066cm}{1.573cm}}{\pgfqpoint{0.04cm}{1.547cm}}
\pgfpathcurveto{\pgfqpoint{0.014cm}{1.522cm}}{\pgfqpoint{0cm}{1.487cm}}{\pgfqpoint{0cm}{1.451cm}}
\pgfpathcurveto{\pgfqpoint{0cm}{1.414cm}}{\pgfqpoint{0.014cm}{1.379cm}}{\pgfqpoint{0.04cm}{1.354cm}}
\pgfpathcurveto{\pgfqpoint{0.066cm}{1.328cm}}{\pgfqpoint{0.1cm}{1.314cm}}{\pgfqpoint{0.137cm}{1.314cm}}
\pgfpathcurveto{\pgfqpoint{0.173cm}{1.314cm}}{\pgfqpoint{0.207cm}{1.328cm}}{\pgfqpoint{0.233cm}{1.354cm}}
\pgfpathcurveto{\pgfqpoint{0.259cm}{1.379cm}}{\pgfqpoint{0.273cm}{1.414cm}}{\pgfqpoint{0.273cm}{1.451cm}}
\pgfusepath{fill}
\pgfpathmoveto{\pgfqpoint{1.345cm}{1.426cm}}
\pgfpathcurveto{\pgfqpoint{1.345cm}{1.463cm}}{\pgfqpoint{1.331cm}{1.497cm}}{\pgfqpoint{1.305cm}{1.523cm}}
\pgfpathcurveto{\pgfqpoint{1.28cm}{1.549cm}}{\pgfqpoint{1.245cm}{1.563cm}}{\pgfqpoint{1.209cm}{1.563cm}}
\pgfpathcurveto{\pgfqpoint{1.172cm}{1.563cm}}{\pgfqpoint{1.138cm}{1.549cm}}{\pgfqpoint{1.112cm}{1.523cm}}
\pgfpathcurveto{\pgfqpoint{1.087cm}{1.497cm}}{\pgfqpoint{1.072cm}{1.463cm}}{\pgfqpoint{1.072cm}{1.426cm}}
\pgfpathcurveto{\pgfqpoint{1.072cm}{1.39cm}}{\pgfqpoint{1.087cm}{1.355cm}}{\pgfqpoint{1.112cm}{1.329cm}}
\pgfpathcurveto{\pgfqpoint{1.138cm}{1.304cm}}{\pgfqpoint{1.172cm}{1.289cm}}{\pgfqpoint{1.209cm}{1.289cm}}
\pgfpathcurveto{\pgfqpoint{1.245cm}{1.289cm}}{\pgfqpoint{1.28cm}{1.304cm}}{\pgfqpoint{1.305cm}{1.329cm}}
\pgfpathcurveto{\pgfqpoint{1.331cm}{1.355cm}}{\pgfqpoint{1.345cm}{1.39cm}}{\pgfqpoint{1.345cm}{1.426cm}}
\pgfusepath{fill}
\begin{pgfscope}
\pgfsetdash{}{0cm}
\pgfsetlinewidth{0.818mm}
\pgfsetroundcap
\pgfsetmiterlimit{4.0}
\pgfpathmoveto{\pgfqpoint{0.682cm}{0.726cm}}
\pgfpathlineto{\pgfqpoint{0.682cm}{0.097cm}}
\pgfusepath{stroke}
\end{pgfscope}
\end{pgfscope}
\end{pgfscope}
\end{pgfscope}
\end{tikzpicture}}}]\|_{C\CC^{\gamma}(\rho^{3+\gamma})}\\
&\quad\lesssim \|\llbracket X^{2} \rrbracket\|_{C\CC^{-1-\kappa}(\rho^{\sigma})}\|-X^{\!\resizebox{0.6em}{!}{
\begin{tikzpicture}
\pgfpathmoveto{\pgfqpoint{0cm}{-0.035cm}}
\pgfpathlineto{\pgfqpoint{1.376cm}{-0.035cm}}
\pgfpathlineto{\pgfqpoint{1.376cm}{1.552cm}}
\pgfpathlineto{\pgfqpoint{0cm}{1.552cm}}
\pgfpathclose
\pgfusepath{clip}
\begin{pgfscope}
\begin{pgfscope}
\pgfpathmoveto{\pgfqpoint{0cm}{-0.035cm}}
\pgfpathlineto{\pgfqpoint{1.376cm}{-0.035cm}}
\pgfpathlineto{\pgfqpoint{1.376cm}{1.552cm}}
\pgfpathlineto{\pgfqpoint{0cm}{1.552cm}}
\pgfpathclose
\pgfusepath{clip}
\begin{pgfscope}
\begin{pgfscope}
\pgfsetdash{}{0cm}
\pgfsetlinewidth{0.818mm}
\pgfsetroundcap
\pgfsetroundjoin
\pgfsetmiterlimit{7.0}
\definecolor{eps2pgf_color}{gray}{0}\pgfsetstrokecolor{eps2pgf_color}\pgfsetfillcolor{eps2pgf_color}
\pgfpathmoveto{\pgfqpoint{0.117cm}{1.421cm}}
\pgfpathlineto{\pgfqpoint{0.682cm}{0.671cm}}
\pgfpathlineto{\pgfqpoint{1.246cm}{1.421cm}}
\pgfusepath{stroke}
\end{pgfscope}
\definecolor{eps2pgf_color}{gray}{0}\pgfsetstrokecolor{eps2pgf_color}\pgfsetfillcolor{eps2pgf_color}
\pgfpathmoveto{\pgfqpoint{0.273cm}{1.395cm}}
\pgfpathcurveto{\pgfqpoint{0.273cm}{1.432cm}}{\pgfqpoint{0.259cm}{1.467cm}}{\pgfqpoint{0.233cm}{1.492cm}}
\pgfpathcurveto{\pgfqpoint{0.207cm}{1.518cm}}{\pgfqpoint{0.173cm}{1.532cm}}{\pgfqpoint{0.137cm}{1.532cm}}
\pgfpathcurveto{\pgfqpoint{0.1cm}{1.532cm}}{\pgfqpoint{0.066cm}{1.518cm}}{\pgfqpoint{0.04cm}{1.492cm}}
\pgfpathcurveto{\pgfqpoint{0.014cm}{1.467cm}}{\pgfqpoint{0cm}{1.432cm}}{\pgfqpoint{0cm}{1.395cm}}
\pgfpathcurveto{\pgfqpoint{0cm}{1.359cm}}{\pgfqpoint{0.014cm}{1.324cm}}{\pgfqpoint{0.04cm}{1.299cm}}
\pgfpathcurveto{\pgfqpoint{0.066cm}{1.273cm}}{\pgfqpoint{0.1cm}{1.258cm}}{\pgfqpoint{0.137cm}{1.258cm}}
\pgfpathcurveto{\pgfqpoint{0.173cm}{1.258cm}}{\pgfqpoint{0.207cm}{1.273cm}}{\pgfqpoint{0.233cm}{1.299cm}}
\pgfpathcurveto{\pgfqpoint{0.259cm}{1.324cm}}{\pgfqpoint{0.273cm}{1.359cm}}{\pgfqpoint{0.273cm}{1.395cm}}
\pgfusepath{fill}
\begin{pgfscope}
\pgfsetdash{}{0cm}
\pgfsetlinewidth{0.818mm}
\pgfsetmiterlimit{7.0}
\pgfpathmoveto{\pgfqpoint{0.682cm}{0.671cm}}
\pgfpathlineto{\pgfqpoint{0.679cm}{1.418cm}}
\pgfusepath{stroke}
\end{pgfscope}
\pgfpathmoveto{\pgfqpoint{0.815cm}{1.399cm}}
\pgfpathcurveto{\pgfqpoint{0.815cm}{1.435cm}}{\pgfqpoint{0.801cm}{1.47cm}}{\pgfqpoint{0.775cm}{1.496cm}}
\pgfpathcurveto{\pgfqpoint{0.75cm}{1.521cm}}{\pgfqpoint{0.715cm}{1.536cm}}{\pgfqpoint{0.679cm}{1.536cm}}
\pgfpathcurveto{\pgfqpoint{0.643cm}{1.536cm}}{\pgfqpoint{0.608cm}{1.521cm}}{\pgfqpoint{0.582cm}{1.496cm}}
\pgfpathcurveto{\pgfqpoint{0.557cm}{1.47cm}}{\pgfqpoint{0.542cm}{1.435cm}}{\pgfqpoint{0.542cm}{1.399cm}}
\pgfpathcurveto{\pgfqpoint{0.542cm}{1.363cm}}{\pgfqpoint{0.557cm}{1.328cm}}{\pgfqpoint{0.582cm}{1.302cm}}
\pgfpathcurveto{\pgfqpoint{0.608cm}{1.276cm}}{\pgfqpoint{0.643cm}{1.262cm}}{\pgfqpoint{0.679cm}{1.262cm}}
\pgfpathcurveto{\pgfqpoint{0.715cm}{1.262cm}}{\pgfqpoint{0.75cm}{1.276cm}}{\pgfqpoint{0.775cm}{1.302cm}}
\pgfpathcurveto{\pgfqpoint{0.801cm}{1.328cm}}{\pgfqpoint{0.815cm}{1.363cm}}{\pgfqpoint{0.815cm}{1.399cm}}
\pgfusepath{fill}
\pgfpathmoveto{\pgfqpoint{1.345cm}{1.371cm}}
\pgfpathcurveto{\pgfqpoint{1.345cm}{1.408cm}}{\pgfqpoint{1.331cm}{1.442cm}}{\pgfqpoint{1.305cm}{1.468cm}}
\pgfpathcurveto{\pgfqpoint{1.28cm}{1.494cm}}{\pgfqpoint{1.245cm}{1.508cm}}{\pgfqpoint{1.209cm}{1.508cm}}
\pgfpathcurveto{\pgfqpoint{1.172cm}{1.508cm}}{\pgfqpoint{1.138cm}{1.494cm}}{\pgfqpoint{1.112cm}{1.468cm}}
\pgfpathcurveto{\pgfqpoint{1.087cm}{1.442cm}}{\pgfqpoint{1.072cm}{1.408cm}}{\pgfqpoint{1.072cm}{1.371cm}}
\pgfpathcurveto{\pgfqpoint{1.072cm}{1.335cm}}{\pgfqpoint{1.087cm}{1.3cm}}{\pgfqpoint{1.112cm}{1.274cm}}
\pgfpathcurveto{\pgfqpoint{1.138cm}{1.249cm}}{\pgfqpoint{1.172cm}{1.234cm}}{\pgfqpoint{1.209cm}{1.234cm}}
\pgfpathcurveto{\pgfqpoint{1.245cm}{1.234cm}}{\pgfqpoint{1.28cm}{1.249cm}}{\pgfqpoint{1.305cm}{1.274cm}}
\pgfpathcurveto{\pgfqpoint{1.331cm}{1.3cm}}{\pgfqpoint{1.345cm}{1.335cm}}{\pgfqpoint{1.345cm}{1.371cm}}
\pgfusepath{fill}
\begin{pgfscope}
\pgfsetdash{}{0cm}
\pgfsetlinewidth{0.818mm}
\pgfsetroundcap
\pgfsetmiterlimit{4.0}
\pgfpathmoveto{\pgfqpoint{0.682cm}{0.671cm}}
\pgfpathlineto{\pgfqpoint{0.682cm}{0.042cm}}
\pgfusepath{stroke}
\end{pgfscope}
\end{pgfscope}
\end{pgfscope}
\end{pgfscope}
\end{tikzpicture}}}+\phi+\psi\|_{C^{(\alpha+\kappa)/2}L^{\infty}(\rho^{3+\gamma'})} \|X^{\!\resizebox{0.6em}{!}{
\begin{tikzpicture}
\pgfpathmoveto{\pgfqpoint{0cm}{0cm}}
\pgfpathlineto{\pgfqpoint{1.376cm}{0cm}}
\pgfpathlineto{\pgfqpoint{1.376cm}{1.588cm}}
\pgfpathlineto{\pgfqpoint{0cm}{1.588cm}}
\pgfpathclose
\pgfusepath{clip}
\begin{pgfscope}
\begin{pgfscope}
\pgfpathmoveto{\pgfqpoint{0cm}{0cm}}
\pgfpathlineto{\pgfqpoint{1.376cm}{0cm}}
\pgfpathlineto{\pgfqpoint{1.376cm}{1.588cm}}
\pgfpathlineto{\pgfqpoint{0cm}{1.588cm}}
\pgfpathclose
\pgfusepath{clip}
\begin{pgfscope}
\begin{pgfscope}
\definecolor{eps2pgf_color}{gray}{0.976471}\pgfsetstrokecolor{eps2pgf_color}\pgfsetfillcolor{eps2pgf_color}
\pgfpathmoveto{\pgfqpoint{0cm}{0cm}}
\pgfpathlineto{\pgfqpoint{1.376cm}{0cm}}
\pgfpathlineto{\pgfqpoint{1.376cm}{1.588cm}}
\pgfpathlineto{\pgfqpoint{0cm}{1.588cm}}
\pgfpathclose
\pgfusepath{fill}
\end{pgfscope}
\begin{pgfscope}
\pgfsetdash{}{0cm}
\pgfsetlinewidth{0.818mm}
\pgfsetroundcap
\pgfsetroundjoin
\pgfsetmiterlimit{7.0}
\definecolor{eps2pgf_color}{gray}{0}\pgfsetstrokecolor{eps2pgf_color}\pgfsetfillcolor{eps2pgf_color}
\pgfpathmoveto{\pgfqpoint{0.117cm}{1.476cm}}
\pgfpathlineto{\pgfqpoint{0.682cm}{0.726cm}}
\pgfpathlineto{\pgfqpoint{1.246cm}{1.476cm}}
\pgfusepath{stroke}
\end{pgfscope}
\definecolor{eps2pgf_color}{gray}{0}\pgfsetstrokecolor{eps2pgf_color}\pgfsetfillcolor{eps2pgf_color}
\pgfpathmoveto{\pgfqpoint{0.273cm}{1.451cm}}
\pgfpathcurveto{\pgfqpoint{0.273cm}{1.487cm}}{\pgfqpoint{0.259cm}{1.522cm}}{\pgfqpoint{0.233cm}{1.547cm}}
\pgfpathcurveto{\pgfqpoint{0.207cm}{1.573cm}}{\pgfqpoint{0.173cm}{1.588cm}}{\pgfqpoint{0.137cm}{1.588cm}}
\pgfpathcurveto{\pgfqpoint{0.1cm}{1.588cm}}{\pgfqpoint{0.066cm}{1.573cm}}{\pgfqpoint{0.04cm}{1.547cm}}
\pgfpathcurveto{\pgfqpoint{0.014cm}{1.522cm}}{\pgfqpoint{0cm}{1.487cm}}{\pgfqpoint{0cm}{1.451cm}}
\pgfpathcurveto{\pgfqpoint{0cm}{1.414cm}}{\pgfqpoint{0.014cm}{1.379cm}}{\pgfqpoint{0.04cm}{1.354cm}}
\pgfpathcurveto{\pgfqpoint{0.066cm}{1.328cm}}{\pgfqpoint{0.1cm}{1.314cm}}{\pgfqpoint{0.137cm}{1.314cm}}
\pgfpathcurveto{\pgfqpoint{0.173cm}{1.314cm}}{\pgfqpoint{0.207cm}{1.328cm}}{\pgfqpoint{0.233cm}{1.354cm}}
\pgfpathcurveto{\pgfqpoint{0.259cm}{1.379cm}}{\pgfqpoint{0.273cm}{1.414cm}}{\pgfqpoint{0.273cm}{1.451cm}}
\pgfusepath{fill}
\pgfpathmoveto{\pgfqpoint{1.345cm}{1.426cm}}
\pgfpathcurveto{\pgfqpoint{1.345cm}{1.463cm}}{\pgfqpoint{1.331cm}{1.497cm}}{\pgfqpoint{1.305cm}{1.523cm}}
\pgfpathcurveto{\pgfqpoint{1.28cm}{1.549cm}}{\pgfqpoint{1.245cm}{1.563cm}}{\pgfqpoint{1.209cm}{1.563cm}}
\pgfpathcurveto{\pgfqpoint{1.172cm}{1.563cm}}{\pgfqpoint{1.138cm}{1.549cm}}{\pgfqpoint{1.112cm}{1.523cm}}
\pgfpathcurveto{\pgfqpoint{1.087cm}{1.497cm}}{\pgfqpoint{1.072cm}{1.463cm}}{\pgfqpoint{1.072cm}{1.426cm}}
\pgfpathcurveto{\pgfqpoint{1.072cm}{1.39cm}}{\pgfqpoint{1.087cm}{1.355cm}}{\pgfqpoint{1.112cm}{1.329cm}}
\pgfpathcurveto{\pgfqpoint{1.138cm}{1.304cm}}{\pgfqpoint{1.172cm}{1.289cm}}{\pgfqpoint{1.209cm}{1.289cm}}
\pgfpathcurveto{\pgfqpoint{1.245cm}{1.289cm}}{\pgfqpoint{1.28cm}{1.304cm}}{\pgfqpoint{1.305cm}{1.329cm}}
\pgfpathcurveto{\pgfqpoint{1.331cm}{1.355cm}}{\pgfqpoint{1.345cm}{1.39cm}}{\pgfqpoint{1.345cm}{1.426cm}}
\pgfusepath{fill}
\begin{pgfscope}
\pgfsetdash{}{0cm}
\pgfsetlinewidth{0.818mm}
\pgfsetroundcap
\pgfsetmiterlimit{4.0}
\pgfpathmoveto{\pgfqpoint{0.682cm}{0.726cm}}
\pgfpathlineto{\pgfqpoint{0.682cm}{0.097cm}}
\pgfusepath{stroke}
\end{pgfscope}
\end{pgfscope}
\end{pgfscope}
\end{pgfscope}
\end{tikzpicture}}}\|_{C\CC^{1-\kappa}(\rho^{\sigma})}\\
&\quad\lesssim 1+\|\phi+\psi\|_{C^{(\alpha+\kappa)/2}L^{\infty}(\rho^{3+\gamma'})}.
 \end{align*}

 The equation for $\vartheta$ now reads as
 \begin{align*}
 \LL \vartheta + \Theta = 0,\quad\vartheta(0)=\phi_{0},
 \end{align*}
with
\begin{align*}
\Theta&=\rmm{\Phi}+\rmm{3\llbracket X^2 \rrbracket\succX^{\!\resizebox{0.6em}{!}{
\begin{tikzpicture}
\pgfpathmoveto{\pgfqpoint{0cm}{-0.035cm}}
\pgfpathlineto{\pgfqpoint{1.376cm}{-0.035cm}}
\pgfpathlineto{\pgfqpoint{1.376cm}{1.552cm}}
\pgfpathlineto{\pgfqpoint{0cm}{1.552cm}}
\pgfpathclose
\pgfusepath{clip}
\begin{pgfscope}
\begin{pgfscope}
\pgfpathmoveto{\pgfqpoint{0cm}{-0.035cm}}
\pgfpathlineto{\pgfqpoint{1.376cm}{-0.035cm}}
\pgfpathlineto{\pgfqpoint{1.376cm}{1.552cm}}
\pgfpathlineto{\pgfqpoint{0cm}{1.552cm}}
\pgfpathclose
\pgfusepath{clip}
\begin{pgfscope}
\begin{pgfscope}
\pgfsetdash{}{0cm}
\pgfsetlinewidth{0.818mm}
\pgfsetroundcap
\pgfsetroundjoin
\pgfsetmiterlimit{7.0}
\definecolor{eps2pgf_color}{gray}{0}\pgfsetstrokecolor{eps2pgf_color}\pgfsetfillcolor{eps2pgf_color}
\pgfpathmoveto{\pgfqpoint{0.117cm}{1.421cm}}
\pgfpathlineto{\pgfqpoint{0.682cm}{0.671cm}}
\pgfpathlineto{\pgfqpoint{1.246cm}{1.421cm}}
\pgfusepath{stroke}
\end{pgfscope}
\definecolor{eps2pgf_color}{gray}{0}\pgfsetstrokecolor{eps2pgf_color}\pgfsetfillcolor{eps2pgf_color}
\pgfpathmoveto{\pgfqpoint{0.273cm}{1.395cm}}
\pgfpathcurveto{\pgfqpoint{0.273cm}{1.432cm}}{\pgfqpoint{0.259cm}{1.467cm}}{\pgfqpoint{0.233cm}{1.492cm}}
\pgfpathcurveto{\pgfqpoint{0.207cm}{1.518cm}}{\pgfqpoint{0.173cm}{1.532cm}}{\pgfqpoint{0.137cm}{1.532cm}}
\pgfpathcurveto{\pgfqpoint{0.1cm}{1.532cm}}{\pgfqpoint{0.066cm}{1.518cm}}{\pgfqpoint{0.04cm}{1.492cm}}
\pgfpathcurveto{\pgfqpoint{0.014cm}{1.467cm}}{\pgfqpoint{0cm}{1.432cm}}{\pgfqpoint{0cm}{1.395cm}}
\pgfpathcurveto{\pgfqpoint{0cm}{1.359cm}}{\pgfqpoint{0.014cm}{1.324cm}}{\pgfqpoint{0.04cm}{1.299cm}}
\pgfpathcurveto{\pgfqpoint{0.066cm}{1.273cm}}{\pgfqpoint{0.1cm}{1.258cm}}{\pgfqpoint{0.137cm}{1.258cm}}
\pgfpathcurveto{\pgfqpoint{0.173cm}{1.258cm}}{\pgfqpoint{0.207cm}{1.273cm}}{\pgfqpoint{0.233cm}{1.299cm}}
\pgfpathcurveto{\pgfqpoint{0.259cm}{1.324cm}}{\pgfqpoint{0.273cm}{1.359cm}}{\pgfqpoint{0.273cm}{1.395cm}}
\pgfusepath{fill}
\begin{pgfscope}
\pgfsetdash{}{0cm}
\pgfsetlinewidth{0.818mm}
\pgfsetmiterlimit{7.0}
\pgfpathmoveto{\pgfqpoint{0.682cm}{0.671cm}}
\pgfpathlineto{\pgfqpoint{0.679cm}{1.418cm}}
\pgfusepath{stroke}
\end{pgfscope}
\pgfpathmoveto{\pgfqpoint{0.815cm}{1.399cm}}
\pgfpathcurveto{\pgfqpoint{0.815cm}{1.435cm}}{\pgfqpoint{0.801cm}{1.47cm}}{\pgfqpoint{0.775cm}{1.496cm}}
\pgfpathcurveto{\pgfqpoint{0.75cm}{1.521cm}}{\pgfqpoint{0.715cm}{1.536cm}}{\pgfqpoint{0.679cm}{1.536cm}}
\pgfpathcurveto{\pgfqpoint{0.643cm}{1.536cm}}{\pgfqpoint{0.608cm}{1.521cm}}{\pgfqpoint{0.582cm}{1.496cm}}
\pgfpathcurveto{\pgfqpoint{0.557cm}{1.47cm}}{\pgfqpoint{0.542cm}{1.435cm}}{\pgfqpoint{0.542cm}{1.399cm}}
\pgfpathcurveto{\pgfqpoint{0.542cm}{1.363cm}}{\pgfqpoint{0.557cm}{1.328cm}}{\pgfqpoint{0.582cm}{1.302cm}}
\pgfpathcurveto{\pgfqpoint{0.608cm}{1.276cm}}{\pgfqpoint{0.643cm}{1.262cm}}{\pgfqpoint{0.679cm}{1.262cm}}
\pgfpathcurveto{\pgfqpoint{0.715cm}{1.262cm}}{\pgfqpoint{0.75cm}{1.276cm}}{\pgfqpoint{0.775cm}{1.302cm}}
\pgfpathcurveto{\pgfqpoint{0.801cm}{1.328cm}}{\pgfqpoint{0.815cm}{1.363cm}}{\pgfqpoint{0.815cm}{1.399cm}}
\pgfusepath{fill}
\pgfpathmoveto{\pgfqpoint{1.345cm}{1.371cm}}
\pgfpathcurveto{\pgfqpoint{1.345cm}{1.408cm}}{\pgfqpoint{1.331cm}{1.442cm}}{\pgfqpoint{1.305cm}{1.468cm}}
\pgfpathcurveto{\pgfqpoint{1.28cm}{1.494cm}}{\pgfqpoint{1.245cm}{1.508cm}}{\pgfqpoint{1.209cm}{1.508cm}}
\pgfpathcurveto{\pgfqpoint{1.172cm}{1.508cm}}{\pgfqpoint{1.138cm}{1.494cm}}{\pgfqpoint{1.112cm}{1.468cm}}
\pgfpathcurveto{\pgfqpoint{1.087cm}{1.442cm}}{\pgfqpoint{1.072cm}{1.408cm}}{\pgfqpoint{1.072cm}{1.371cm}}
\pgfpathcurveto{\pgfqpoint{1.072cm}{1.335cm}}{\pgfqpoint{1.087cm}{1.3cm}}{\pgfqpoint{1.112cm}{1.274cm}}
\pgfpathcurveto{\pgfqpoint{1.138cm}{1.249cm}}{\pgfqpoint{1.172cm}{1.234cm}}{\pgfqpoint{1.209cm}{1.234cm}}
\pgfpathcurveto{\pgfqpoint{1.245cm}{1.234cm}}{\pgfqpoint{1.28cm}{1.249cm}}{\pgfqpoint{1.305cm}{1.274cm}}
\pgfpathcurveto{\pgfqpoint{1.331cm}{1.3cm}}{\pgfqpoint{1.345cm}{1.335cm}}{\pgfqpoint{1.345cm}{1.371cm}}
\pgfusepath{fill}
\begin{pgfscope}
\pgfsetdash{}{0cm}
\pgfsetlinewidth{0.818mm}
\pgfsetroundcap
\pgfsetmiterlimit{4.0}
\pgfpathmoveto{\pgfqpoint{0.682cm}{0.671cm}}
\pgfpathlineto{\pgfqpoint{0.682cm}{0.042cm}}
\pgfusepath{stroke}
\end{pgfscope}
\end{pgfscope}
\end{pgfscope}
\end{pgfscope}
\end{tikzpicture}}}}-\rmm{3\VV_>\llbracket X^2 \rrbracket\succ(\phi+\psi)}+\rmb{3\VV_\leq\llbracket X^2 \rrbracket\succ(\phi+\psi)}\\
&\quad -3\Big( (-X^{\!\resizebox{0.6em}{!}{
\begin{tikzpicture}
\pgfpathmoveto{\pgfqpoint{0cm}{-0.035cm}}
\pgfpathlineto{\pgfqpoint{1.376cm}{-0.035cm}}
\pgfpathlineto{\pgfqpoint{1.376cm}{1.552cm}}
\pgfpathlineto{\pgfqpoint{0cm}{1.552cm}}
\pgfpathclose
\pgfusepath{clip}
\begin{pgfscope}
\begin{pgfscope}
\pgfpathmoveto{\pgfqpoint{0cm}{-0.035cm}}
\pgfpathlineto{\pgfqpoint{1.376cm}{-0.035cm}}
\pgfpathlineto{\pgfqpoint{1.376cm}{1.552cm}}
\pgfpathlineto{\pgfqpoint{0cm}{1.552cm}}
\pgfpathclose
\pgfusepath{clip}
\begin{pgfscope}
\begin{pgfscope}
\pgfsetdash{}{0cm}
\pgfsetlinewidth{0.818mm}
\pgfsetroundcap
\pgfsetroundjoin
\pgfsetmiterlimit{7.0}
\definecolor{eps2pgf_color}{gray}{0}\pgfsetstrokecolor{eps2pgf_color}\pgfsetfillcolor{eps2pgf_color}
\pgfpathmoveto{\pgfqpoint{0.117cm}{1.421cm}}
\pgfpathlineto{\pgfqpoint{0.682cm}{0.671cm}}
\pgfpathlineto{\pgfqpoint{1.246cm}{1.421cm}}
\pgfusepath{stroke}
\end{pgfscope}
\definecolor{eps2pgf_color}{gray}{0}\pgfsetstrokecolor{eps2pgf_color}\pgfsetfillcolor{eps2pgf_color}
\pgfpathmoveto{\pgfqpoint{0.273cm}{1.395cm}}
\pgfpathcurveto{\pgfqpoint{0.273cm}{1.432cm}}{\pgfqpoint{0.259cm}{1.467cm}}{\pgfqpoint{0.233cm}{1.492cm}}
\pgfpathcurveto{\pgfqpoint{0.207cm}{1.518cm}}{\pgfqpoint{0.173cm}{1.532cm}}{\pgfqpoint{0.137cm}{1.532cm}}
\pgfpathcurveto{\pgfqpoint{0.1cm}{1.532cm}}{\pgfqpoint{0.066cm}{1.518cm}}{\pgfqpoint{0.04cm}{1.492cm}}
\pgfpathcurveto{\pgfqpoint{0.014cm}{1.467cm}}{\pgfqpoint{0cm}{1.432cm}}{\pgfqpoint{0cm}{1.395cm}}
\pgfpathcurveto{\pgfqpoint{0cm}{1.359cm}}{\pgfqpoint{0.014cm}{1.324cm}}{\pgfqpoint{0.04cm}{1.299cm}}
\pgfpathcurveto{\pgfqpoint{0.066cm}{1.273cm}}{\pgfqpoint{0.1cm}{1.258cm}}{\pgfqpoint{0.137cm}{1.258cm}}
\pgfpathcurveto{\pgfqpoint{0.173cm}{1.258cm}}{\pgfqpoint{0.207cm}{1.273cm}}{\pgfqpoint{0.233cm}{1.299cm}}
\pgfpathcurveto{\pgfqpoint{0.259cm}{1.324cm}}{\pgfqpoint{0.273cm}{1.359cm}}{\pgfqpoint{0.273cm}{1.395cm}}
\pgfusepath{fill}
\begin{pgfscope}
\pgfsetdash{}{0cm}
\pgfsetlinewidth{0.818mm}
\pgfsetmiterlimit{7.0}
\pgfpathmoveto{\pgfqpoint{0.682cm}{0.671cm}}
\pgfpathlineto{\pgfqpoint{0.679cm}{1.418cm}}
\pgfusepath{stroke}
\end{pgfscope}
\pgfpathmoveto{\pgfqpoint{0.815cm}{1.399cm}}
\pgfpathcurveto{\pgfqpoint{0.815cm}{1.435cm}}{\pgfqpoint{0.801cm}{1.47cm}}{\pgfqpoint{0.775cm}{1.496cm}}
\pgfpathcurveto{\pgfqpoint{0.75cm}{1.521cm}}{\pgfqpoint{0.715cm}{1.536cm}}{\pgfqpoint{0.679cm}{1.536cm}}
\pgfpathcurveto{\pgfqpoint{0.643cm}{1.536cm}}{\pgfqpoint{0.608cm}{1.521cm}}{\pgfqpoint{0.582cm}{1.496cm}}
\pgfpathcurveto{\pgfqpoint{0.557cm}{1.47cm}}{\pgfqpoint{0.542cm}{1.435cm}}{\pgfqpoint{0.542cm}{1.399cm}}
\pgfpathcurveto{\pgfqpoint{0.542cm}{1.363cm}}{\pgfqpoint{0.557cm}{1.328cm}}{\pgfqpoint{0.582cm}{1.302cm}}
\pgfpathcurveto{\pgfqpoint{0.608cm}{1.276cm}}{\pgfqpoint{0.643cm}{1.262cm}}{\pgfqpoint{0.679cm}{1.262cm}}
\pgfpathcurveto{\pgfqpoint{0.715cm}{1.262cm}}{\pgfqpoint{0.75cm}{1.276cm}}{\pgfqpoint{0.775cm}{1.302cm}}
\pgfpathcurveto{\pgfqpoint{0.801cm}{1.328cm}}{\pgfqpoint{0.815cm}{1.363cm}}{\pgfqpoint{0.815cm}{1.399cm}}
\pgfusepath{fill}
\pgfpathmoveto{\pgfqpoint{1.345cm}{1.371cm}}
\pgfpathcurveto{\pgfqpoint{1.345cm}{1.408cm}}{\pgfqpoint{1.331cm}{1.442cm}}{\pgfqpoint{1.305cm}{1.468cm}}
\pgfpathcurveto{\pgfqpoint{1.28cm}{1.494cm}}{\pgfqpoint{1.245cm}{1.508cm}}{\pgfqpoint{1.209cm}{1.508cm}}
\pgfpathcurveto{\pgfqpoint{1.172cm}{1.508cm}}{\pgfqpoint{1.138cm}{1.494cm}}{\pgfqpoint{1.112cm}{1.468cm}}
\pgfpathcurveto{\pgfqpoint{1.087cm}{1.442cm}}{\pgfqpoint{1.072cm}{1.408cm}}{\pgfqpoint{1.072cm}{1.371cm}}
\pgfpathcurveto{\pgfqpoint{1.072cm}{1.335cm}}{\pgfqpoint{1.087cm}{1.3cm}}{\pgfqpoint{1.112cm}{1.274cm}}
\pgfpathcurveto{\pgfqpoint{1.138cm}{1.249cm}}{\pgfqpoint{1.172cm}{1.234cm}}{\pgfqpoint{1.209cm}{1.234cm}}
\pgfpathcurveto{\pgfqpoint{1.245cm}{1.234cm}}{\pgfqpoint{1.28cm}{1.249cm}}{\pgfqpoint{1.305cm}{1.274cm}}
\pgfpathcurveto{\pgfqpoint{1.331cm}{1.3cm}}{\pgfqpoint{1.345cm}{1.335cm}}{\pgfqpoint{1.345cm}{1.371cm}}
\pgfusepath{fill}
\begin{pgfscope}
\pgfsetdash{}{0cm}
\pgfsetlinewidth{0.818mm}
\pgfsetroundcap
\pgfsetmiterlimit{4.0}
\pgfpathmoveto{\pgfqpoint{0.682cm}{0.671cm}}
\pgfpathlineto{\pgfqpoint{0.682cm}{0.042cm}}
\pgfusepath{stroke}
\end{pgfscope}
\end{pgfscope}
\end{pgfscope}
\end{pgfscope}
\end{tikzpicture}}}+\phi+\psi)\Prec\llbracket X^{2}\rrbracket- (-X^{\!\resizebox{0.6em}{!}{
\begin{tikzpicture}
\pgfpathmoveto{\pgfqpoint{0cm}{-0.035cm}}
\pgfpathlineto{\pgfqpoint{1.376cm}{-0.035cm}}
\pgfpathlineto{\pgfqpoint{1.376cm}{1.552cm}}
\pgfpathlineto{\pgfqpoint{0cm}{1.552cm}}
\pgfpathclose
\pgfusepath{clip}
\begin{pgfscope}
\begin{pgfscope}
\pgfpathmoveto{\pgfqpoint{0cm}{-0.035cm}}
\pgfpathlineto{\pgfqpoint{1.376cm}{-0.035cm}}
\pgfpathlineto{\pgfqpoint{1.376cm}{1.552cm}}
\pgfpathlineto{\pgfqpoint{0cm}{1.552cm}}
\pgfpathclose
\pgfusepath{clip}
\begin{pgfscope}
\begin{pgfscope}
\pgfsetdash{}{0cm}
\pgfsetlinewidth{0.818mm}
\pgfsetroundcap
\pgfsetroundjoin
\pgfsetmiterlimit{7.0}
\definecolor{eps2pgf_color}{gray}{0}\pgfsetstrokecolor{eps2pgf_color}\pgfsetfillcolor{eps2pgf_color}
\pgfpathmoveto{\pgfqpoint{0.117cm}{1.421cm}}
\pgfpathlineto{\pgfqpoint{0.682cm}{0.671cm}}
\pgfpathlineto{\pgfqpoint{1.246cm}{1.421cm}}
\pgfusepath{stroke}
\end{pgfscope}
\definecolor{eps2pgf_color}{gray}{0}\pgfsetstrokecolor{eps2pgf_color}\pgfsetfillcolor{eps2pgf_color}
\pgfpathmoveto{\pgfqpoint{0.273cm}{1.395cm}}
\pgfpathcurveto{\pgfqpoint{0.273cm}{1.432cm}}{\pgfqpoint{0.259cm}{1.467cm}}{\pgfqpoint{0.233cm}{1.492cm}}
\pgfpathcurveto{\pgfqpoint{0.207cm}{1.518cm}}{\pgfqpoint{0.173cm}{1.532cm}}{\pgfqpoint{0.137cm}{1.532cm}}
\pgfpathcurveto{\pgfqpoint{0.1cm}{1.532cm}}{\pgfqpoint{0.066cm}{1.518cm}}{\pgfqpoint{0.04cm}{1.492cm}}
\pgfpathcurveto{\pgfqpoint{0.014cm}{1.467cm}}{\pgfqpoint{0cm}{1.432cm}}{\pgfqpoint{0cm}{1.395cm}}
\pgfpathcurveto{\pgfqpoint{0cm}{1.359cm}}{\pgfqpoint{0.014cm}{1.324cm}}{\pgfqpoint{0.04cm}{1.299cm}}
\pgfpathcurveto{\pgfqpoint{0.066cm}{1.273cm}}{\pgfqpoint{0.1cm}{1.258cm}}{\pgfqpoint{0.137cm}{1.258cm}}
\pgfpathcurveto{\pgfqpoint{0.173cm}{1.258cm}}{\pgfqpoint{0.207cm}{1.273cm}}{\pgfqpoint{0.233cm}{1.299cm}}
\pgfpathcurveto{\pgfqpoint{0.259cm}{1.324cm}}{\pgfqpoint{0.273cm}{1.359cm}}{\pgfqpoint{0.273cm}{1.395cm}}
\pgfusepath{fill}
\begin{pgfscope}
\pgfsetdash{}{0cm}
\pgfsetlinewidth{0.818mm}
\pgfsetmiterlimit{7.0}
\pgfpathmoveto{\pgfqpoint{0.682cm}{0.671cm}}
\pgfpathlineto{\pgfqpoint{0.679cm}{1.418cm}}
\pgfusepath{stroke}
\end{pgfscope}
\pgfpathmoveto{\pgfqpoint{0.815cm}{1.399cm}}
\pgfpathcurveto{\pgfqpoint{0.815cm}{1.435cm}}{\pgfqpoint{0.801cm}{1.47cm}}{\pgfqpoint{0.775cm}{1.496cm}}
\pgfpathcurveto{\pgfqpoint{0.75cm}{1.521cm}}{\pgfqpoint{0.715cm}{1.536cm}}{\pgfqpoint{0.679cm}{1.536cm}}
\pgfpathcurveto{\pgfqpoint{0.643cm}{1.536cm}}{\pgfqpoint{0.608cm}{1.521cm}}{\pgfqpoint{0.582cm}{1.496cm}}
\pgfpathcurveto{\pgfqpoint{0.557cm}{1.47cm}}{\pgfqpoint{0.542cm}{1.435cm}}{\pgfqpoint{0.542cm}{1.399cm}}
\pgfpathcurveto{\pgfqpoint{0.542cm}{1.363cm}}{\pgfqpoint{0.557cm}{1.328cm}}{\pgfqpoint{0.582cm}{1.302cm}}
\pgfpathcurveto{\pgfqpoint{0.608cm}{1.276cm}}{\pgfqpoint{0.643cm}{1.262cm}}{\pgfqpoint{0.679cm}{1.262cm}}
\pgfpathcurveto{\pgfqpoint{0.715cm}{1.262cm}}{\pgfqpoint{0.75cm}{1.276cm}}{\pgfqpoint{0.775cm}{1.302cm}}
\pgfpathcurveto{\pgfqpoint{0.801cm}{1.328cm}}{\pgfqpoint{0.815cm}{1.363cm}}{\pgfqpoint{0.815cm}{1.399cm}}
\pgfusepath{fill}
\pgfpathmoveto{\pgfqpoint{1.345cm}{1.371cm}}
\pgfpathcurveto{\pgfqpoint{1.345cm}{1.408cm}}{\pgfqpoint{1.331cm}{1.442cm}}{\pgfqpoint{1.305cm}{1.468cm}}
\pgfpathcurveto{\pgfqpoint{1.28cm}{1.494cm}}{\pgfqpoint{1.245cm}{1.508cm}}{\pgfqpoint{1.209cm}{1.508cm}}
\pgfpathcurveto{\pgfqpoint{1.172cm}{1.508cm}}{\pgfqpoint{1.138cm}{1.494cm}}{\pgfqpoint{1.112cm}{1.468cm}}
\pgfpathcurveto{\pgfqpoint{1.087cm}{1.442cm}}{\pgfqpoint{1.072cm}{1.408cm}}{\pgfqpoint{1.072cm}{1.371cm}}
\pgfpathcurveto{\pgfqpoint{1.072cm}{1.335cm}}{\pgfqpoint{1.087cm}{1.3cm}}{\pgfqpoint{1.112cm}{1.274cm}}
\pgfpathcurveto{\pgfqpoint{1.138cm}{1.249cm}}{\pgfqpoint{1.172cm}{1.234cm}}{\pgfqpoint{1.209cm}{1.234cm}}
\pgfpathcurveto{\pgfqpoint{1.245cm}{1.234cm}}{\pgfqpoint{1.28cm}{1.249cm}}{\pgfqpoint{1.305cm}{1.274cm}}
\pgfpathcurveto{\pgfqpoint{1.331cm}{1.3cm}}{\pgfqpoint{1.345cm}{1.335cm}}{\pgfqpoint{1.345cm}{1.371cm}}
\pgfusepath{fill}
\begin{pgfscope}
\pgfsetdash{}{0cm}
\pgfsetlinewidth{0.818mm}
\pgfsetroundcap
\pgfsetmiterlimit{4.0}
\pgfpathmoveto{\pgfqpoint{0.682cm}{0.671cm}}
\pgfpathlineto{\pgfqpoint{0.682cm}{0.042cm}}
\pgfusepath{stroke}
\end{pgfscope}
\end{pgfscope}
\end{pgfscope}
\end{pgfscope}
\end{tikzpicture}}}+\phi+\psi)\prec\llbracket X^{2}\rrbracket\Big)-3[\LL,(-X^{\!\resizebox{0.6em}{!}{
\begin{tikzpicture}
\pgfpathmoveto{\pgfqpoint{0cm}{-0.035cm}}
\pgfpathlineto{\pgfqpoint{1.376cm}{-0.035cm}}
\pgfpathlineto{\pgfqpoint{1.376cm}{1.552cm}}
\pgfpathlineto{\pgfqpoint{0cm}{1.552cm}}
\pgfpathclose
\pgfusepath{clip}
\begin{pgfscope}
\begin{pgfscope}
\pgfpathmoveto{\pgfqpoint{0cm}{-0.035cm}}
\pgfpathlineto{\pgfqpoint{1.376cm}{-0.035cm}}
\pgfpathlineto{\pgfqpoint{1.376cm}{1.552cm}}
\pgfpathlineto{\pgfqpoint{0cm}{1.552cm}}
\pgfpathclose
\pgfusepath{clip}
\begin{pgfscope}
\begin{pgfscope}
\pgfsetdash{}{0cm}
\pgfsetlinewidth{0.818mm}
\pgfsetroundcap
\pgfsetroundjoin
\pgfsetmiterlimit{7.0}
\definecolor{eps2pgf_color}{gray}{0}\pgfsetstrokecolor{eps2pgf_color}\pgfsetfillcolor{eps2pgf_color}
\pgfpathmoveto{\pgfqpoint{0.117cm}{1.421cm}}
\pgfpathlineto{\pgfqpoint{0.682cm}{0.671cm}}
\pgfpathlineto{\pgfqpoint{1.246cm}{1.421cm}}
\pgfusepath{stroke}
\end{pgfscope}
\definecolor{eps2pgf_color}{gray}{0}\pgfsetstrokecolor{eps2pgf_color}\pgfsetfillcolor{eps2pgf_color}
\pgfpathmoveto{\pgfqpoint{0.273cm}{1.395cm}}
\pgfpathcurveto{\pgfqpoint{0.273cm}{1.432cm}}{\pgfqpoint{0.259cm}{1.467cm}}{\pgfqpoint{0.233cm}{1.492cm}}
\pgfpathcurveto{\pgfqpoint{0.207cm}{1.518cm}}{\pgfqpoint{0.173cm}{1.532cm}}{\pgfqpoint{0.137cm}{1.532cm}}
\pgfpathcurveto{\pgfqpoint{0.1cm}{1.532cm}}{\pgfqpoint{0.066cm}{1.518cm}}{\pgfqpoint{0.04cm}{1.492cm}}
\pgfpathcurveto{\pgfqpoint{0.014cm}{1.467cm}}{\pgfqpoint{0cm}{1.432cm}}{\pgfqpoint{0cm}{1.395cm}}
\pgfpathcurveto{\pgfqpoint{0cm}{1.359cm}}{\pgfqpoint{0.014cm}{1.324cm}}{\pgfqpoint{0.04cm}{1.299cm}}
\pgfpathcurveto{\pgfqpoint{0.066cm}{1.273cm}}{\pgfqpoint{0.1cm}{1.258cm}}{\pgfqpoint{0.137cm}{1.258cm}}
\pgfpathcurveto{\pgfqpoint{0.173cm}{1.258cm}}{\pgfqpoint{0.207cm}{1.273cm}}{\pgfqpoint{0.233cm}{1.299cm}}
\pgfpathcurveto{\pgfqpoint{0.259cm}{1.324cm}}{\pgfqpoint{0.273cm}{1.359cm}}{\pgfqpoint{0.273cm}{1.395cm}}
\pgfusepath{fill}
\begin{pgfscope}
\pgfsetdash{}{0cm}
\pgfsetlinewidth{0.818mm}
\pgfsetmiterlimit{7.0}
\pgfpathmoveto{\pgfqpoint{0.682cm}{0.671cm}}
\pgfpathlineto{\pgfqpoint{0.679cm}{1.418cm}}
\pgfusepath{stroke}
\end{pgfscope}
\pgfpathmoveto{\pgfqpoint{0.815cm}{1.399cm}}
\pgfpathcurveto{\pgfqpoint{0.815cm}{1.435cm}}{\pgfqpoint{0.801cm}{1.47cm}}{\pgfqpoint{0.775cm}{1.496cm}}
\pgfpathcurveto{\pgfqpoint{0.75cm}{1.521cm}}{\pgfqpoint{0.715cm}{1.536cm}}{\pgfqpoint{0.679cm}{1.536cm}}
\pgfpathcurveto{\pgfqpoint{0.643cm}{1.536cm}}{\pgfqpoint{0.608cm}{1.521cm}}{\pgfqpoint{0.582cm}{1.496cm}}
\pgfpathcurveto{\pgfqpoint{0.557cm}{1.47cm}}{\pgfqpoint{0.542cm}{1.435cm}}{\pgfqpoint{0.542cm}{1.399cm}}
\pgfpathcurveto{\pgfqpoint{0.542cm}{1.363cm}}{\pgfqpoint{0.557cm}{1.328cm}}{\pgfqpoint{0.582cm}{1.302cm}}
\pgfpathcurveto{\pgfqpoint{0.608cm}{1.276cm}}{\pgfqpoint{0.643cm}{1.262cm}}{\pgfqpoint{0.679cm}{1.262cm}}
\pgfpathcurveto{\pgfqpoint{0.715cm}{1.262cm}}{\pgfqpoint{0.75cm}{1.276cm}}{\pgfqpoint{0.775cm}{1.302cm}}
\pgfpathcurveto{\pgfqpoint{0.801cm}{1.328cm}}{\pgfqpoint{0.815cm}{1.363cm}}{\pgfqpoint{0.815cm}{1.399cm}}
\pgfusepath{fill}
\pgfpathmoveto{\pgfqpoint{1.345cm}{1.371cm}}
\pgfpathcurveto{\pgfqpoint{1.345cm}{1.408cm}}{\pgfqpoint{1.331cm}{1.442cm}}{\pgfqpoint{1.305cm}{1.468cm}}
\pgfpathcurveto{\pgfqpoint{1.28cm}{1.494cm}}{\pgfqpoint{1.245cm}{1.508cm}}{\pgfqpoint{1.209cm}{1.508cm}}
\pgfpathcurveto{\pgfqpoint{1.172cm}{1.508cm}}{\pgfqpoint{1.138cm}{1.494cm}}{\pgfqpoint{1.112cm}{1.468cm}}
\pgfpathcurveto{\pgfqpoint{1.087cm}{1.442cm}}{\pgfqpoint{1.072cm}{1.408cm}}{\pgfqpoint{1.072cm}{1.371cm}}
\pgfpathcurveto{\pgfqpoint{1.072cm}{1.335cm}}{\pgfqpoint{1.087cm}{1.3cm}}{\pgfqpoint{1.112cm}{1.274cm}}
\pgfpathcurveto{\pgfqpoint{1.138cm}{1.249cm}}{\pgfqpoint{1.172cm}{1.234cm}}{\pgfqpoint{1.209cm}{1.234cm}}
\pgfpathcurveto{\pgfqpoint{1.245cm}{1.234cm}}{\pgfqpoint{1.28cm}{1.249cm}}{\pgfqpoint{1.305cm}{1.274cm}}
\pgfpathcurveto{\pgfqpoint{1.331cm}{1.3cm}}{\pgfqpoint{1.345cm}{1.335cm}}{\pgfqpoint{1.345cm}{1.371cm}}
\pgfusepath{fill}
\begin{pgfscope}
\pgfsetdash{}{0cm}
\pgfsetlinewidth{0.818mm}
\pgfsetroundcap
\pgfsetmiterlimit{4.0}
\pgfpathmoveto{\pgfqpoint{0.682cm}{0.671cm}}
\pgfpathlineto{\pgfqpoint{0.682cm}{0.042cm}}
\pgfusepath{stroke}
\end{pgfscope}
\end{pgfscope}
\end{pgfscope}
\end{pgfscope}
\end{tikzpicture}}}+\phi+\psi)\Prec]X^{\!\resizebox{0.6em}{!}{
\begin{tikzpicture}
\pgfpathmoveto{\pgfqpoint{0cm}{0cm}}
\pgfpathlineto{\pgfqpoint{1.376cm}{0cm}}
\pgfpathlineto{\pgfqpoint{1.376cm}{1.588cm}}
\pgfpathlineto{\pgfqpoint{0cm}{1.588cm}}
\pgfpathclose
\pgfusepath{clip}
\begin{pgfscope}
\begin{pgfscope}
\pgfpathmoveto{\pgfqpoint{0cm}{0cm}}
\pgfpathlineto{\pgfqpoint{1.376cm}{0cm}}
\pgfpathlineto{\pgfqpoint{1.376cm}{1.588cm}}
\pgfpathlineto{\pgfqpoint{0cm}{1.588cm}}
\pgfpathclose
\pgfusepath{clip}
\begin{pgfscope}
\begin{pgfscope}
\definecolor{eps2pgf_color}{gray}{0.976471}\pgfsetstrokecolor{eps2pgf_color}\pgfsetfillcolor{eps2pgf_color}
\pgfpathmoveto{\pgfqpoint{0cm}{0cm}}
\pgfpathlineto{\pgfqpoint{1.376cm}{0cm}}
\pgfpathlineto{\pgfqpoint{1.376cm}{1.588cm}}
\pgfpathlineto{\pgfqpoint{0cm}{1.588cm}}
\pgfpathclose
\pgfusepath{fill}
\end{pgfscope}
\begin{pgfscope}
\pgfsetdash{}{0cm}
\pgfsetlinewidth{0.818mm}
\pgfsetroundcap
\pgfsetroundjoin
\pgfsetmiterlimit{7.0}
\definecolor{eps2pgf_color}{gray}{0}\pgfsetstrokecolor{eps2pgf_color}\pgfsetfillcolor{eps2pgf_color}
\pgfpathmoveto{\pgfqpoint{0.117cm}{1.476cm}}
\pgfpathlineto{\pgfqpoint{0.682cm}{0.726cm}}
\pgfpathlineto{\pgfqpoint{1.246cm}{1.476cm}}
\pgfusepath{stroke}
\end{pgfscope}
\definecolor{eps2pgf_color}{gray}{0}\pgfsetstrokecolor{eps2pgf_color}\pgfsetfillcolor{eps2pgf_color}
\pgfpathmoveto{\pgfqpoint{0.273cm}{1.451cm}}
\pgfpathcurveto{\pgfqpoint{0.273cm}{1.487cm}}{\pgfqpoint{0.259cm}{1.522cm}}{\pgfqpoint{0.233cm}{1.547cm}}
\pgfpathcurveto{\pgfqpoint{0.207cm}{1.573cm}}{\pgfqpoint{0.173cm}{1.588cm}}{\pgfqpoint{0.137cm}{1.588cm}}
\pgfpathcurveto{\pgfqpoint{0.1cm}{1.588cm}}{\pgfqpoint{0.066cm}{1.573cm}}{\pgfqpoint{0.04cm}{1.547cm}}
\pgfpathcurveto{\pgfqpoint{0.014cm}{1.522cm}}{\pgfqpoint{0cm}{1.487cm}}{\pgfqpoint{0cm}{1.451cm}}
\pgfpathcurveto{\pgfqpoint{0cm}{1.414cm}}{\pgfqpoint{0.014cm}{1.379cm}}{\pgfqpoint{0.04cm}{1.354cm}}
\pgfpathcurveto{\pgfqpoint{0.066cm}{1.328cm}}{\pgfqpoint{0.1cm}{1.314cm}}{\pgfqpoint{0.137cm}{1.314cm}}
\pgfpathcurveto{\pgfqpoint{0.173cm}{1.314cm}}{\pgfqpoint{0.207cm}{1.328cm}}{\pgfqpoint{0.233cm}{1.354cm}}
\pgfpathcurveto{\pgfqpoint{0.259cm}{1.379cm}}{\pgfqpoint{0.273cm}{1.414cm}}{\pgfqpoint{0.273cm}{1.451cm}}
\pgfusepath{fill}
\pgfpathmoveto{\pgfqpoint{1.345cm}{1.426cm}}
\pgfpathcurveto{\pgfqpoint{1.345cm}{1.463cm}}{\pgfqpoint{1.331cm}{1.497cm}}{\pgfqpoint{1.305cm}{1.523cm}}
\pgfpathcurveto{\pgfqpoint{1.28cm}{1.549cm}}{\pgfqpoint{1.245cm}{1.563cm}}{\pgfqpoint{1.209cm}{1.563cm}}
\pgfpathcurveto{\pgfqpoint{1.172cm}{1.563cm}}{\pgfqpoint{1.138cm}{1.549cm}}{\pgfqpoint{1.112cm}{1.523cm}}
\pgfpathcurveto{\pgfqpoint{1.087cm}{1.497cm}}{\pgfqpoint{1.072cm}{1.463cm}}{\pgfqpoint{1.072cm}{1.426cm}}
\pgfpathcurveto{\pgfqpoint{1.072cm}{1.39cm}}{\pgfqpoint{1.087cm}{1.355cm}}{\pgfqpoint{1.112cm}{1.329cm}}
\pgfpathcurveto{\pgfqpoint{1.138cm}{1.304cm}}{\pgfqpoint{1.172cm}{1.289cm}}{\pgfqpoint{1.209cm}{1.289cm}}
\pgfpathcurveto{\pgfqpoint{1.245cm}{1.289cm}}{\pgfqpoint{1.28cm}{1.304cm}}{\pgfqpoint{1.305cm}{1.329cm}}
\pgfpathcurveto{\pgfqpoint{1.331cm}{1.355cm}}{\pgfqpoint{1.345cm}{1.39cm}}{\pgfqpoint{1.345cm}{1.426cm}}
\pgfusepath{fill}
\begin{pgfscope}
\pgfsetdash{}{0cm}
\pgfsetlinewidth{0.818mm}
\pgfsetroundcap
\pgfsetmiterlimit{4.0}
\pgfpathmoveto{\pgfqpoint{0.682cm}{0.726cm}}
\pgfpathlineto{\pgfqpoint{0.682cm}{0.097cm}}
\pgfusepath{stroke}
\end{pgfscope}
\end{pgfscope}
\end{pgfscope}
\end{pgfscope}
\end{tikzpicture}}}.
\end{align*}
Here the two new terms are estimated using Lemma \ref{lem:5.1} as follows
\begin{align*}
&3\big\| (-X^{\!\resizebox{0.6em}{!}{
\begin{tikzpicture}
\pgfpathmoveto{\pgfqpoint{0cm}{-0.035cm}}
\pgfpathlineto{\pgfqpoint{1.376cm}{-0.035cm}}
\pgfpathlineto{\pgfqpoint{1.376cm}{1.552cm}}
\pgfpathlineto{\pgfqpoint{0cm}{1.552cm}}
\pgfpathclose
\pgfusepath{clip}
\begin{pgfscope}
\begin{pgfscope}
\pgfpathmoveto{\pgfqpoint{0cm}{-0.035cm}}
\pgfpathlineto{\pgfqpoint{1.376cm}{-0.035cm}}
\pgfpathlineto{\pgfqpoint{1.376cm}{1.552cm}}
\pgfpathlineto{\pgfqpoint{0cm}{1.552cm}}
\pgfpathclose
\pgfusepath{clip}
\begin{pgfscope}
\begin{pgfscope}
\pgfsetdash{}{0cm}
\pgfsetlinewidth{0.818mm}
\pgfsetroundcap
\pgfsetroundjoin
\pgfsetmiterlimit{7.0}
\definecolor{eps2pgf_color}{gray}{0}\pgfsetstrokecolor{eps2pgf_color}\pgfsetfillcolor{eps2pgf_color}
\pgfpathmoveto{\pgfqpoint{0.117cm}{1.421cm}}
\pgfpathlineto{\pgfqpoint{0.682cm}{0.671cm}}
\pgfpathlineto{\pgfqpoint{1.246cm}{1.421cm}}
\pgfusepath{stroke}
\end{pgfscope}
\definecolor{eps2pgf_color}{gray}{0}\pgfsetstrokecolor{eps2pgf_color}\pgfsetfillcolor{eps2pgf_color}
\pgfpathmoveto{\pgfqpoint{0.273cm}{1.395cm}}
\pgfpathcurveto{\pgfqpoint{0.273cm}{1.432cm}}{\pgfqpoint{0.259cm}{1.467cm}}{\pgfqpoint{0.233cm}{1.492cm}}
\pgfpathcurveto{\pgfqpoint{0.207cm}{1.518cm}}{\pgfqpoint{0.173cm}{1.532cm}}{\pgfqpoint{0.137cm}{1.532cm}}
\pgfpathcurveto{\pgfqpoint{0.1cm}{1.532cm}}{\pgfqpoint{0.066cm}{1.518cm}}{\pgfqpoint{0.04cm}{1.492cm}}
\pgfpathcurveto{\pgfqpoint{0.014cm}{1.467cm}}{\pgfqpoint{0cm}{1.432cm}}{\pgfqpoint{0cm}{1.395cm}}
\pgfpathcurveto{\pgfqpoint{0cm}{1.359cm}}{\pgfqpoint{0.014cm}{1.324cm}}{\pgfqpoint{0.04cm}{1.299cm}}
\pgfpathcurveto{\pgfqpoint{0.066cm}{1.273cm}}{\pgfqpoint{0.1cm}{1.258cm}}{\pgfqpoint{0.137cm}{1.258cm}}
\pgfpathcurveto{\pgfqpoint{0.173cm}{1.258cm}}{\pgfqpoint{0.207cm}{1.273cm}}{\pgfqpoint{0.233cm}{1.299cm}}
\pgfpathcurveto{\pgfqpoint{0.259cm}{1.324cm}}{\pgfqpoint{0.273cm}{1.359cm}}{\pgfqpoint{0.273cm}{1.395cm}}
\pgfusepath{fill}
\begin{pgfscope}
\pgfsetdash{}{0cm}
\pgfsetlinewidth{0.818mm}
\pgfsetmiterlimit{7.0}
\pgfpathmoveto{\pgfqpoint{0.682cm}{0.671cm}}
\pgfpathlineto{\pgfqpoint{0.679cm}{1.418cm}}
\pgfusepath{stroke}
\end{pgfscope}
\pgfpathmoveto{\pgfqpoint{0.815cm}{1.399cm}}
\pgfpathcurveto{\pgfqpoint{0.815cm}{1.435cm}}{\pgfqpoint{0.801cm}{1.47cm}}{\pgfqpoint{0.775cm}{1.496cm}}
\pgfpathcurveto{\pgfqpoint{0.75cm}{1.521cm}}{\pgfqpoint{0.715cm}{1.536cm}}{\pgfqpoint{0.679cm}{1.536cm}}
\pgfpathcurveto{\pgfqpoint{0.643cm}{1.536cm}}{\pgfqpoint{0.608cm}{1.521cm}}{\pgfqpoint{0.582cm}{1.496cm}}
\pgfpathcurveto{\pgfqpoint{0.557cm}{1.47cm}}{\pgfqpoint{0.542cm}{1.435cm}}{\pgfqpoint{0.542cm}{1.399cm}}
\pgfpathcurveto{\pgfqpoint{0.542cm}{1.363cm}}{\pgfqpoint{0.557cm}{1.328cm}}{\pgfqpoint{0.582cm}{1.302cm}}
\pgfpathcurveto{\pgfqpoint{0.608cm}{1.276cm}}{\pgfqpoint{0.643cm}{1.262cm}}{\pgfqpoint{0.679cm}{1.262cm}}
\pgfpathcurveto{\pgfqpoint{0.715cm}{1.262cm}}{\pgfqpoint{0.75cm}{1.276cm}}{\pgfqpoint{0.775cm}{1.302cm}}
\pgfpathcurveto{\pgfqpoint{0.801cm}{1.328cm}}{\pgfqpoint{0.815cm}{1.363cm}}{\pgfqpoint{0.815cm}{1.399cm}}
\pgfusepath{fill}
\pgfpathmoveto{\pgfqpoint{1.345cm}{1.371cm}}
\pgfpathcurveto{\pgfqpoint{1.345cm}{1.408cm}}{\pgfqpoint{1.331cm}{1.442cm}}{\pgfqpoint{1.305cm}{1.468cm}}
\pgfpathcurveto{\pgfqpoint{1.28cm}{1.494cm}}{\pgfqpoint{1.245cm}{1.508cm}}{\pgfqpoint{1.209cm}{1.508cm}}
\pgfpathcurveto{\pgfqpoint{1.172cm}{1.508cm}}{\pgfqpoint{1.138cm}{1.494cm}}{\pgfqpoint{1.112cm}{1.468cm}}
\pgfpathcurveto{\pgfqpoint{1.087cm}{1.442cm}}{\pgfqpoint{1.072cm}{1.408cm}}{\pgfqpoint{1.072cm}{1.371cm}}
\pgfpathcurveto{\pgfqpoint{1.072cm}{1.335cm}}{\pgfqpoint{1.087cm}{1.3cm}}{\pgfqpoint{1.112cm}{1.274cm}}
\pgfpathcurveto{\pgfqpoint{1.138cm}{1.249cm}}{\pgfqpoint{1.172cm}{1.234cm}}{\pgfqpoint{1.209cm}{1.234cm}}
\pgfpathcurveto{\pgfqpoint{1.245cm}{1.234cm}}{\pgfqpoint{1.28cm}{1.249cm}}{\pgfqpoint{1.305cm}{1.274cm}}
\pgfpathcurveto{\pgfqpoint{1.331cm}{1.3cm}}{\pgfqpoint{1.345cm}{1.335cm}}{\pgfqpoint{1.345cm}{1.371cm}}
\pgfusepath{fill}
\begin{pgfscope}
\pgfsetdash{}{0cm}
\pgfsetlinewidth{0.818mm}
\pgfsetroundcap
\pgfsetmiterlimit{4.0}
\pgfpathmoveto{\pgfqpoint{0.682cm}{0.671cm}}
\pgfpathlineto{\pgfqpoint{0.682cm}{0.042cm}}
\pgfusepath{stroke}
\end{pgfscope}
\end{pgfscope}
\end{pgfscope}
\end{pgfscope}
\end{tikzpicture}}}+\phi+\psi)\Prec\llbracket X^{2}\rrbracket- (-X^{\!\resizebox{0.6em}{!}{
\begin{tikzpicture}
\pgfpathmoveto{\pgfqpoint{0cm}{-0.035cm}}
\pgfpathlineto{\pgfqpoint{1.376cm}{-0.035cm}}
\pgfpathlineto{\pgfqpoint{1.376cm}{1.552cm}}
\pgfpathlineto{\pgfqpoint{0cm}{1.552cm}}
\pgfpathclose
\pgfusepath{clip}
\begin{pgfscope}
\begin{pgfscope}
\pgfpathmoveto{\pgfqpoint{0cm}{-0.035cm}}
\pgfpathlineto{\pgfqpoint{1.376cm}{-0.035cm}}
\pgfpathlineto{\pgfqpoint{1.376cm}{1.552cm}}
\pgfpathlineto{\pgfqpoint{0cm}{1.552cm}}
\pgfpathclose
\pgfusepath{clip}
\begin{pgfscope}
\begin{pgfscope}
\pgfsetdash{}{0cm}
\pgfsetlinewidth{0.818mm}
\pgfsetroundcap
\pgfsetroundjoin
\pgfsetmiterlimit{7.0}
\definecolor{eps2pgf_color}{gray}{0}\pgfsetstrokecolor{eps2pgf_color}\pgfsetfillcolor{eps2pgf_color}
\pgfpathmoveto{\pgfqpoint{0.117cm}{1.421cm}}
\pgfpathlineto{\pgfqpoint{0.682cm}{0.671cm}}
\pgfpathlineto{\pgfqpoint{1.246cm}{1.421cm}}
\pgfusepath{stroke}
\end{pgfscope}
\definecolor{eps2pgf_color}{gray}{0}\pgfsetstrokecolor{eps2pgf_color}\pgfsetfillcolor{eps2pgf_color}
\pgfpathmoveto{\pgfqpoint{0.273cm}{1.395cm}}
\pgfpathcurveto{\pgfqpoint{0.273cm}{1.432cm}}{\pgfqpoint{0.259cm}{1.467cm}}{\pgfqpoint{0.233cm}{1.492cm}}
\pgfpathcurveto{\pgfqpoint{0.207cm}{1.518cm}}{\pgfqpoint{0.173cm}{1.532cm}}{\pgfqpoint{0.137cm}{1.532cm}}
\pgfpathcurveto{\pgfqpoint{0.1cm}{1.532cm}}{\pgfqpoint{0.066cm}{1.518cm}}{\pgfqpoint{0.04cm}{1.492cm}}
\pgfpathcurveto{\pgfqpoint{0.014cm}{1.467cm}}{\pgfqpoint{0cm}{1.432cm}}{\pgfqpoint{0cm}{1.395cm}}
\pgfpathcurveto{\pgfqpoint{0cm}{1.359cm}}{\pgfqpoint{0.014cm}{1.324cm}}{\pgfqpoint{0.04cm}{1.299cm}}
\pgfpathcurveto{\pgfqpoint{0.066cm}{1.273cm}}{\pgfqpoint{0.1cm}{1.258cm}}{\pgfqpoint{0.137cm}{1.258cm}}
\pgfpathcurveto{\pgfqpoint{0.173cm}{1.258cm}}{\pgfqpoint{0.207cm}{1.273cm}}{\pgfqpoint{0.233cm}{1.299cm}}
\pgfpathcurveto{\pgfqpoint{0.259cm}{1.324cm}}{\pgfqpoint{0.273cm}{1.359cm}}{\pgfqpoint{0.273cm}{1.395cm}}
\pgfusepath{fill}
\begin{pgfscope}
\pgfsetdash{}{0cm}
\pgfsetlinewidth{0.818mm}
\pgfsetmiterlimit{7.0}
\pgfpathmoveto{\pgfqpoint{0.682cm}{0.671cm}}
\pgfpathlineto{\pgfqpoint{0.679cm}{1.418cm}}
\pgfusepath{stroke}
\end{pgfscope}
\pgfpathmoveto{\pgfqpoint{0.815cm}{1.399cm}}
\pgfpathcurveto{\pgfqpoint{0.815cm}{1.435cm}}{\pgfqpoint{0.801cm}{1.47cm}}{\pgfqpoint{0.775cm}{1.496cm}}
\pgfpathcurveto{\pgfqpoint{0.75cm}{1.521cm}}{\pgfqpoint{0.715cm}{1.536cm}}{\pgfqpoint{0.679cm}{1.536cm}}
\pgfpathcurveto{\pgfqpoint{0.643cm}{1.536cm}}{\pgfqpoint{0.608cm}{1.521cm}}{\pgfqpoint{0.582cm}{1.496cm}}
\pgfpathcurveto{\pgfqpoint{0.557cm}{1.47cm}}{\pgfqpoint{0.542cm}{1.435cm}}{\pgfqpoint{0.542cm}{1.399cm}}
\pgfpathcurveto{\pgfqpoint{0.542cm}{1.363cm}}{\pgfqpoint{0.557cm}{1.328cm}}{\pgfqpoint{0.582cm}{1.302cm}}
\pgfpathcurveto{\pgfqpoint{0.608cm}{1.276cm}}{\pgfqpoint{0.643cm}{1.262cm}}{\pgfqpoint{0.679cm}{1.262cm}}
\pgfpathcurveto{\pgfqpoint{0.715cm}{1.262cm}}{\pgfqpoint{0.75cm}{1.276cm}}{\pgfqpoint{0.775cm}{1.302cm}}
\pgfpathcurveto{\pgfqpoint{0.801cm}{1.328cm}}{\pgfqpoint{0.815cm}{1.363cm}}{\pgfqpoint{0.815cm}{1.399cm}}
\pgfusepath{fill}
\pgfpathmoveto{\pgfqpoint{1.345cm}{1.371cm}}
\pgfpathcurveto{\pgfqpoint{1.345cm}{1.408cm}}{\pgfqpoint{1.331cm}{1.442cm}}{\pgfqpoint{1.305cm}{1.468cm}}
\pgfpathcurveto{\pgfqpoint{1.28cm}{1.494cm}}{\pgfqpoint{1.245cm}{1.508cm}}{\pgfqpoint{1.209cm}{1.508cm}}
\pgfpathcurveto{\pgfqpoint{1.172cm}{1.508cm}}{\pgfqpoint{1.138cm}{1.494cm}}{\pgfqpoint{1.112cm}{1.468cm}}
\pgfpathcurveto{\pgfqpoint{1.087cm}{1.442cm}}{\pgfqpoint{1.072cm}{1.408cm}}{\pgfqpoint{1.072cm}{1.371cm}}
\pgfpathcurveto{\pgfqpoint{1.072cm}{1.335cm}}{\pgfqpoint{1.087cm}{1.3cm}}{\pgfqpoint{1.112cm}{1.274cm}}
\pgfpathcurveto{\pgfqpoint{1.138cm}{1.249cm}}{\pgfqpoint{1.172cm}{1.234cm}}{\pgfqpoint{1.209cm}{1.234cm}}
\pgfpathcurveto{\pgfqpoint{1.245cm}{1.234cm}}{\pgfqpoint{1.28cm}{1.249cm}}{\pgfqpoint{1.305cm}{1.274cm}}
\pgfpathcurveto{\pgfqpoint{1.331cm}{1.3cm}}{\pgfqpoint{1.345cm}{1.335cm}}{\pgfqpoint{1.345cm}{1.371cm}}
\pgfusepath{fill}
\begin{pgfscope}
\pgfsetdash{}{0cm}
\pgfsetlinewidth{0.818mm}
\pgfsetroundcap
\pgfsetmiterlimit{4.0}
\pgfpathmoveto{\pgfqpoint{0.682cm}{0.671cm}}
\pgfpathlineto{\pgfqpoint{0.682cm}{0.042cm}}
\pgfusepath{stroke}
\end{pgfscope}
\end{pgfscope}
\end{pgfscope}
\end{pgfscope}
\end{tikzpicture}}}+\phi+\psi)\prec\llbracket X^{2}\rrbracket\big\|_{C\CC^{-1+\alpha}(\rho^{3+\gamma'})}\\
& \quad\lesssim \|-X^{\!\resizebox{0.6em}{!}{
\begin{tikzpicture}
\pgfpathmoveto{\pgfqpoint{0cm}{-0.035cm}}
\pgfpathlineto{\pgfqpoint{1.376cm}{-0.035cm}}
\pgfpathlineto{\pgfqpoint{1.376cm}{1.552cm}}
\pgfpathlineto{\pgfqpoint{0cm}{1.552cm}}
\pgfpathclose
\pgfusepath{clip}
\begin{pgfscope}
\begin{pgfscope}
\pgfpathmoveto{\pgfqpoint{0cm}{-0.035cm}}
\pgfpathlineto{\pgfqpoint{1.376cm}{-0.035cm}}
\pgfpathlineto{\pgfqpoint{1.376cm}{1.552cm}}
\pgfpathlineto{\pgfqpoint{0cm}{1.552cm}}
\pgfpathclose
\pgfusepath{clip}
\begin{pgfscope}
\begin{pgfscope}
\pgfsetdash{}{0cm}
\pgfsetlinewidth{0.818mm}
\pgfsetroundcap
\pgfsetroundjoin
\pgfsetmiterlimit{7.0}
\definecolor{eps2pgf_color}{gray}{0}\pgfsetstrokecolor{eps2pgf_color}\pgfsetfillcolor{eps2pgf_color}
\pgfpathmoveto{\pgfqpoint{0.117cm}{1.421cm}}
\pgfpathlineto{\pgfqpoint{0.682cm}{0.671cm}}
\pgfpathlineto{\pgfqpoint{1.246cm}{1.421cm}}
\pgfusepath{stroke}
\end{pgfscope}
\definecolor{eps2pgf_color}{gray}{0}\pgfsetstrokecolor{eps2pgf_color}\pgfsetfillcolor{eps2pgf_color}
\pgfpathmoveto{\pgfqpoint{0.273cm}{1.395cm}}
\pgfpathcurveto{\pgfqpoint{0.273cm}{1.432cm}}{\pgfqpoint{0.259cm}{1.467cm}}{\pgfqpoint{0.233cm}{1.492cm}}
\pgfpathcurveto{\pgfqpoint{0.207cm}{1.518cm}}{\pgfqpoint{0.173cm}{1.532cm}}{\pgfqpoint{0.137cm}{1.532cm}}
\pgfpathcurveto{\pgfqpoint{0.1cm}{1.532cm}}{\pgfqpoint{0.066cm}{1.518cm}}{\pgfqpoint{0.04cm}{1.492cm}}
\pgfpathcurveto{\pgfqpoint{0.014cm}{1.467cm}}{\pgfqpoint{0cm}{1.432cm}}{\pgfqpoint{0cm}{1.395cm}}
\pgfpathcurveto{\pgfqpoint{0cm}{1.359cm}}{\pgfqpoint{0.014cm}{1.324cm}}{\pgfqpoint{0.04cm}{1.299cm}}
\pgfpathcurveto{\pgfqpoint{0.066cm}{1.273cm}}{\pgfqpoint{0.1cm}{1.258cm}}{\pgfqpoint{0.137cm}{1.258cm}}
\pgfpathcurveto{\pgfqpoint{0.173cm}{1.258cm}}{\pgfqpoint{0.207cm}{1.273cm}}{\pgfqpoint{0.233cm}{1.299cm}}
\pgfpathcurveto{\pgfqpoint{0.259cm}{1.324cm}}{\pgfqpoint{0.273cm}{1.359cm}}{\pgfqpoint{0.273cm}{1.395cm}}
\pgfusepath{fill}
\begin{pgfscope}
\pgfsetdash{}{0cm}
\pgfsetlinewidth{0.818mm}
\pgfsetmiterlimit{7.0}
\pgfpathmoveto{\pgfqpoint{0.682cm}{0.671cm}}
\pgfpathlineto{\pgfqpoint{0.679cm}{1.418cm}}
\pgfusepath{stroke}
\end{pgfscope}
\pgfpathmoveto{\pgfqpoint{0.815cm}{1.399cm}}
\pgfpathcurveto{\pgfqpoint{0.815cm}{1.435cm}}{\pgfqpoint{0.801cm}{1.47cm}}{\pgfqpoint{0.775cm}{1.496cm}}
\pgfpathcurveto{\pgfqpoint{0.75cm}{1.521cm}}{\pgfqpoint{0.715cm}{1.536cm}}{\pgfqpoint{0.679cm}{1.536cm}}
\pgfpathcurveto{\pgfqpoint{0.643cm}{1.536cm}}{\pgfqpoint{0.608cm}{1.521cm}}{\pgfqpoint{0.582cm}{1.496cm}}
\pgfpathcurveto{\pgfqpoint{0.557cm}{1.47cm}}{\pgfqpoint{0.542cm}{1.435cm}}{\pgfqpoint{0.542cm}{1.399cm}}
\pgfpathcurveto{\pgfqpoint{0.542cm}{1.363cm}}{\pgfqpoint{0.557cm}{1.328cm}}{\pgfqpoint{0.582cm}{1.302cm}}
\pgfpathcurveto{\pgfqpoint{0.608cm}{1.276cm}}{\pgfqpoint{0.643cm}{1.262cm}}{\pgfqpoint{0.679cm}{1.262cm}}
\pgfpathcurveto{\pgfqpoint{0.715cm}{1.262cm}}{\pgfqpoint{0.75cm}{1.276cm}}{\pgfqpoint{0.775cm}{1.302cm}}
\pgfpathcurveto{\pgfqpoint{0.801cm}{1.328cm}}{\pgfqpoint{0.815cm}{1.363cm}}{\pgfqpoint{0.815cm}{1.399cm}}
\pgfusepath{fill}
\pgfpathmoveto{\pgfqpoint{1.345cm}{1.371cm}}
\pgfpathcurveto{\pgfqpoint{1.345cm}{1.408cm}}{\pgfqpoint{1.331cm}{1.442cm}}{\pgfqpoint{1.305cm}{1.468cm}}
\pgfpathcurveto{\pgfqpoint{1.28cm}{1.494cm}}{\pgfqpoint{1.245cm}{1.508cm}}{\pgfqpoint{1.209cm}{1.508cm}}
\pgfpathcurveto{\pgfqpoint{1.172cm}{1.508cm}}{\pgfqpoint{1.138cm}{1.494cm}}{\pgfqpoint{1.112cm}{1.468cm}}
\pgfpathcurveto{\pgfqpoint{1.087cm}{1.442cm}}{\pgfqpoint{1.072cm}{1.408cm}}{\pgfqpoint{1.072cm}{1.371cm}}
\pgfpathcurveto{\pgfqpoint{1.072cm}{1.335cm}}{\pgfqpoint{1.087cm}{1.3cm}}{\pgfqpoint{1.112cm}{1.274cm}}
\pgfpathcurveto{\pgfqpoint{1.138cm}{1.249cm}}{\pgfqpoint{1.172cm}{1.234cm}}{\pgfqpoint{1.209cm}{1.234cm}}
\pgfpathcurveto{\pgfqpoint{1.245cm}{1.234cm}}{\pgfqpoint{1.28cm}{1.249cm}}{\pgfqpoint{1.305cm}{1.274cm}}
\pgfpathcurveto{\pgfqpoint{1.331cm}{1.3cm}}{\pgfqpoint{1.345cm}{1.335cm}}{\pgfqpoint{1.345cm}{1.371cm}}
\pgfusepath{fill}
\begin{pgfscope}
\pgfsetdash{}{0cm}
\pgfsetlinewidth{0.818mm}
\pgfsetroundcap
\pgfsetmiterlimit{4.0}
\pgfpathmoveto{\pgfqpoint{0.682cm}{0.671cm}}
\pgfpathlineto{\pgfqpoint{0.682cm}{0.042cm}}
\pgfusepath{stroke}
\end{pgfscope}
\end{pgfscope}
\end{pgfscope}
\end{pgfscope}
\end{tikzpicture}}}+\phi+\psi\|_{C^{(\alpha+\kappa)/2}L^{\infty}(\rho^{ 3+\gamma''})} \|\llbracket X^{2} \rrbracket\|_{C\CC^{-1-\kappa}(\rho^{\sigma})}\lesssim 1+\|\phi+\psi\|_{C^{(\alpha+\kappa)/2}L^{\infty}(\rho^{ 3+\gamma''})},
\end{align*}
\begin{align*}
&3\|[\LL,(-X^{\!\resizebox{0.6em}{!}{
\begin{tikzpicture}
\pgfpathmoveto{\pgfqpoint{0cm}{-0.035cm}}
\pgfpathlineto{\pgfqpoint{1.376cm}{-0.035cm}}
\pgfpathlineto{\pgfqpoint{1.376cm}{1.552cm}}
\pgfpathlineto{\pgfqpoint{0cm}{1.552cm}}
\pgfpathclose
\pgfusepath{clip}
\begin{pgfscope}
\begin{pgfscope}
\pgfpathmoveto{\pgfqpoint{0cm}{-0.035cm}}
\pgfpathlineto{\pgfqpoint{1.376cm}{-0.035cm}}
\pgfpathlineto{\pgfqpoint{1.376cm}{1.552cm}}
\pgfpathlineto{\pgfqpoint{0cm}{1.552cm}}
\pgfpathclose
\pgfusepath{clip}
\begin{pgfscope}
\begin{pgfscope}
\pgfsetdash{}{0cm}
\pgfsetlinewidth{0.818mm}
\pgfsetroundcap
\pgfsetroundjoin
\pgfsetmiterlimit{7.0}
\definecolor{eps2pgf_color}{gray}{0}\pgfsetstrokecolor{eps2pgf_color}\pgfsetfillcolor{eps2pgf_color}
\pgfpathmoveto{\pgfqpoint{0.117cm}{1.421cm}}
\pgfpathlineto{\pgfqpoint{0.682cm}{0.671cm}}
\pgfpathlineto{\pgfqpoint{1.246cm}{1.421cm}}
\pgfusepath{stroke}
\end{pgfscope}
\definecolor{eps2pgf_color}{gray}{0}\pgfsetstrokecolor{eps2pgf_color}\pgfsetfillcolor{eps2pgf_color}
\pgfpathmoveto{\pgfqpoint{0.273cm}{1.395cm}}
\pgfpathcurveto{\pgfqpoint{0.273cm}{1.432cm}}{\pgfqpoint{0.259cm}{1.467cm}}{\pgfqpoint{0.233cm}{1.492cm}}
\pgfpathcurveto{\pgfqpoint{0.207cm}{1.518cm}}{\pgfqpoint{0.173cm}{1.532cm}}{\pgfqpoint{0.137cm}{1.532cm}}
\pgfpathcurveto{\pgfqpoint{0.1cm}{1.532cm}}{\pgfqpoint{0.066cm}{1.518cm}}{\pgfqpoint{0.04cm}{1.492cm}}
\pgfpathcurveto{\pgfqpoint{0.014cm}{1.467cm}}{\pgfqpoint{0cm}{1.432cm}}{\pgfqpoint{0cm}{1.395cm}}
\pgfpathcurveto{\pgfqpoint{0cm}{1.359cm}}{\pgfqpoint{0.014cm}{1.324cm}}{\pgfqpoint{0.04cm}{1.299cm}}
\pgfpathcurveto{\pgfqpoint{0.066cm}{1.273cm}}{\pgfqpoint{0.1cm}{1.258cm}}{\pgfqpoint{0.137cm}{1.258cm}}
\pgfpathcurveto{\pgfqpoint{0.173cm}{1.258cm}}{\pgfqpoint{0.207cm}{1.273cm}}{\pgfqpoint{0.233cm}{1.299cm}}
\pgfpathcurveto{\pgfqpoint{0.259cm}{1.324cm}}{\pgfqpoint{0.273cm}{1.359cm}}{\pgfqpoint{0.273cm}{1.395cm}}
\pgfusepath{fill}
\begin{pgfscope}
\pgfsetdash{}{0cm}
\pgfsetlinewidth{0.818mm}
\pgfsetmiterlimit{7.0}
\pgfpathmoveto{\pgfqpoint{0.682cm}{0.671cm}}
\pgfpathlineto{\pgfqpoint{0.679cm}{1.418cm}}
\pgfusepath{stroke}
\end{pgfscope}
\pgfpathmoveto{\pgfqpoint{0.815cm}{1.399cm}}
\pgfpathcurveto{\pgfqpoint{0.815cm}{1.435cm}}{\pgfqpoint{0.801cm}{1.47cm}}{\pgfqpoint{0.775cm}{1.496cm}}
\pgfpathcurveto{\pgfqpoint{0.75cm}{1.521cm}}{\pgfqpoint{0.715cm}{1.536cm}}{\pgfqpoint{0.679cm}{1.536cm}}
\pgfpathcurveto{\pgfqpoint{0.643cm}{1.536cm}}{\pgfqpoint{0.608cm}{1.521cm}}{\pgfqpoint{0.582cm}{1.496cm}}
\pgfpathcurveto{\pgfqpoint{0.557cm}{1.47cm}}{\pgfqpoint{0.542cm}{1.435cm}}{\pgfqpoint{0.542cm}{1.399cm}}
\pgfpathcurveto{\pgfqpoint{0.542cm}{1.363cm}}{\pgfqpoint{0.557cm}{1.328cm}}{\pgfqpoint{0.582cm}{1.302cm}}
\pgfpathcurveto{\pgfqpoint{0.608cm}{1.276cm}}{\pgfqpoint{0.643cm}{1.262cm}}{\pgfqpoint{0.679cm}{1.262cm}}
\pgfpathcurveto{\pgfqpoint{0.715cm}{1.262cm}}{\pgfqpoint{0.75cm}{1.276cm}}{\pgfqpoint{0.775cm}{1.302cm}}
\pgfpathcurveto{\pgfqpoint{0.801cm}{1.328cm}}{\pgfqpoint{0.815cm}{1.363cm}}{\pgfqpoint{0.815cm}{1.399cm}}
\pgfusepath{fill}
\pgfpathmoveto{\pgfqpoint{1.345cm}{1.371cm}}
\pgfpathcurveto{\pgfqpoint{1.345cm}{1.408cm}}{\pgfqpoint{1.331cm}{1.442cm}}{\pgfqpoint{1.305cm}{1.468cm}}
\pgfpathcurveto{\pgfqpoint{1.28cm}{1.494cm}}{\pgfqpoint{1.245cm}{1.508cm}}{\pgfqpoint{1.209cm}{1.508cm}}
\pgfpathcurveto{\pgfqpoint{1.172cm}{1.508cm}}{\pgfqpoint{1.138cm}{1.494cm}}{\pgfqpoint{1.112cm}{1.468cm}}
\pgfpathcurveto{\pgfqpoint{1.087cm}{1.442cm}}{\pgfqpoint{1.072cm}{1.408cm}}{\pgfqpoint{1.072cm}{1.371cm}}
\pgfpathcurveto{\pgfqpoint{1.072cm}{1.335cm}}{\pgfqpoint{1.087cm}{1.3cm}}{\pgfqpoint{1.112cm}{1.274cm}}
\pgfpathcurveto{\pgfqpoint{1.138cm}{1.249cm}}{\pgfqpoint{1.172cm}{1.234cm}}{\pgfqpoint{1.209cm}{1.234cm}}
\pgfpathcurveto{\pgfqpoint{1.245cm}{1.234cm}}{\pgfqpoint{1.28cm}{1.249cm}}{\pgfqpoint{1.305cm}{1.274cm}}
\pgfpathcurveto{\pgfqpoint{1.331cm}{1.3cm}}{\pgfqpoint{1.345cm}{1.335cm}}{\pgfqpoint{1.345cm}{1.371cm}}
\pgfusepath{fill}
\begin{pgfscope}
\pgfsetdash{}{0cm}
\pgfsetlinewidth{0.818mm}
\pgfsetroundcap
\pgfsetmiterlimit{4.0}
\pgfpathmoveto{\pgfqpoint{0.682cm}{0.671cm}}
\pgfpathlineto{\pgfqpoint{0.682cm}{0.042cm}}
\pgfusepath{stroke}
\end{pgfscope}
\end{pgfscope}
\end{pgfscope}
\end{pgfscope}
\end{tikzpicture}}}+\phi+\psi)\Prec]X^{\!\resizebox{0.6em}{!}{
\begin{tikzpicture}
\pgfpathmoveto{\pgfqpoint{0cm}{0cm}}
\pgfpathlineto{\pgfqpoint{1.376cm}{0cm}}
\pgfpathlineto{\pgfqpoint{1.376cm}{1.588cm}}
\pgfpathlineto{\pgfqpoint{0cm}{1.588cm}}
\pgfpathclose
\pgfusepath{clip}
\begin{pgfscope}
\begin{pgfscope}
\pgfpathmoveto{\pgfqpoint{0cm}{0cm}}
\pgfpathlineto{\pgfqpoint{1.376cm}{0cm}}
\pgfpathlineto{\pgfqpoint{1.376cm}{1.588cm}}
\pgfpathlineto{\pgfqpoint{0cm}{1.588cm}}
\pgfpathclose
\pgfusepath{clip}
\begin{pgfscope}
\begin{pgfscope}
\definecolor{eps2pgf_color}{gray}{0.976471}\pgfsetstrokecolor{eps2pgf_color}\pgfsetfillcolor{eps2pgf_color}
\pgfpathmoveto{\pgfqpoint{0cm}{0cm}}
\pgfpathlineto{\pgfqpoint{1.376cm}{0cm}}
\pgfpathlineto{\pgfqpoint{1.376cm}{1.588cm}}
\pgfpathlineto{\pgfqpoint{0cm}{1.588cm}}
\pgfpathclose
\pgfusepath{fill}
\end{pgfscope}
\begin{pgfscope}
\pgfsetdash{}{0cm}
\pgfsetlinewidth{0.818mm}
\pgfsetroundcap
\pgfsetroundjoin
\pgfsetmiterlimit{7.0}
\definecolor{eps2pgf_color}{gray}{0}\pgfsetstrokecolor{eps2pgf_color}\pgfsetfillcolor{eps2pgf_color}
\pgfpathmoveto{\pgfqpoint{0.117cm}{1.476cm}}
\pgfpathlineto{\pgfqpoint{0.682cm}{0.726cm}}
\pgfpathlineto{\pgfqpoint{1.246cm}{1.476cm}}
\pgfusepath{stroke}
\end{pgfscope}
\definecolor{eps2pgf_color}{gray}{0}\pgfsetstrokecolor{eps2pgf_color}\pgfsetfillcolor{eps2pgf_color}
\pgfpathmoveto{\pgfqpoint{0.273cm}{1.451cm}}
\pgfpathcurveto{\pgfqpoint{0.273cm}{1.487cm}}{\pgfqpoint{0.259cm}{1.522cm}}{\pgfqpoint{0.233cm}{1.547cm}}
\pgfpathcurveto{\pgfqpoint{0.207cm}{1.573cm}}{\pgfqpoint{0.173cm}{1.588cm}}{\pgfqpoint{0.137cm}{1.588cm}}
\pgfpathcurveto{\pgfqpoint{0.1cm}{1.588cm}}{\pgfqpoint{0.066cm}{1.573cm}}{\pgfqpoint{0.04cm}{1.547cm}}
\pgfpathcurveto{\pgfqpoint{0.014cm}{1.522cm}}{\pgfqpoint{0cm}{1.487cm}}{\pgfqpoint{0cm}{1.451cm}}
\pgfpathcurveto{\pgfqpoint{0cm}{1.414cm}}{\pgfqpoint{0.014cm}{1.379cm}}{\pgfqpoint{0.04cm}{1.354cm}}
\pgfpathcurveto{\pgfqpoint{0.066cm}{1.328cm}}{\pgfqpoint{0.1cm}{1.314cm}}{\pgfqpoint{0.137cm}{1.314cm}}
\pgfpathcurveto{\pgfqpoint{0.173cm}{1.314cm}}{\pgfqpoint{0.207cm}{1.328cm}}{\pgfqpoint{0.233cm}{1.354cm}}
\pgfpathcurveto{\pgfqpoint{0.259cm}{1.379cm}}{\pgfqpoint{0.273cm}{1.414cm}}{\pgfqpoint{0.273cm}{1.451cm}}
\pgfusepath{fill}
\pgfpathmoveto{\pgfqpoint{1.345cm}{1.426cm}}
\pgfpathcurveto{\pgfqpoint{1.345cm}{1.463cm}}{\pgfqpoint{1.331cm}{1.497cm}}{\pgfqpoint{1.305cm}{1.523cm}}
\pgfpathcurveto{\pgfqpoint{1.28cm}{1.549cm}}{\pgfqpoint{1.245cm}{1.563cm}}{\pgfqpoint{1.209cm}{1.563cm}}
\pgfpathcurveto{\pgfqpoint{1.172cm}{1.563cm}}{\pgfqpoint{1.138cm}{1.549cm}}{\pgfqpoint{1.112cm}{1.523cm}}
\pgfpathcurveto{\pgfqpoint{1.087cm}{1.497cm}}{\pgfqpoint{1.072cm}{1.463cm}}{\pgfqpoint{1.072cm}{1.426cm}}
\pgfpathcurveto{\pgfqpoint{1.072cm}{1.39cm}}{\pgfqpoint{1.087cm}{1.355cm}}{\pgfqpoint{1.112cm}{1.329cm}}
\pgfpathcurveto{\pgfqpoint{1.138cm}{1.304cm}}{\pgfqpoint{1.172cm}{1.289cm}}{\pgfqpoint{1.209cm}{1.289cm}}
\pgfpathcurveto{\pgfqpoint{1.245cm}{1.289cm}}{\pgfqpoint{1.28cm}{1.304cm}}{\pgfqpoint{1.305cm}{1.329cm}}
\pgfpathcurveto{\pgfqpoint{1.331cm}{1.355cm}}{\pgfqpoint{1.345cm}{1.39cm}}{\pgfqpoint{1.345cm}{1.426cm}}
\pgfusepath{fill}
\begin{pgfscope}
\pgfsetdash{}{0cm}
\pgfsetlinewidth{0.818mm}
\pgfsetroundcap
\pgfsetmiterlimit{4.0}
\pgfpathmoveto{\pgfqpoint{0.682cm}{0.726cm}}
\pgfpathlineto{\pgfqpoint{0.682cm}{0.097cm}}
\pgfusepath{stroke}
\end{pgfscope}
\end{pgfscope}
\end{pgfscope}
\end{pgfscope}
\end{tikzpicture}}}\|_{C\CC^{-1+\alpha}(\rho^{3+\gamma'})}\\
&\quad\lesssim \Big( \|-X^{\!\resizebox{0.6em}{!}{
\begin{tikzpicture}
\pgfpathmoveto{\pgfqpoint{0cm}{-0.035cm}}
\pgfpathlineto{\pgfqpoint{1.376cm}{-0.035cm}}
\pgfpathlineto{\pgfqpoint{1.376cm}{1.552cm}}
\pgfpathlineto{\pgfqpoint{0cm}{1.552cm}}
\pgfpathclose
\pgfusepath{clip}
\begin{pgfscope}
\begin{pgfscope}
\pgfpathmoveto{\pgfqpoint{0cm}{-0.035cm}}
\pgfpathlineto{\pgfqpoint{1.376cm}{-0.035cm}}
\pgfpathlineto{\pgfqpoint{1.376cm}{1.552cm}}
\pgfpathlineto{\pgfqpoint{0cm}{1.552cm}}
\pgfpathclose
\pgfusepath{clip}
\begin{pgfscope}
\begin{pgfscope}
\pgfsetdash{}{0cm}
\pgfsetlinewidth{0.818mm}
\pgfsetroundcap
\pgfsetroundjoin
\pgfsetmiterlimit{7.0}
\definecolor{eps2pgf_color}{gray}{0}\pgfsetstrokecolor{eps2pgf_color}\pgfsetfillcolor{eps2pgf_color}
\pgfpathmoveto{\pgfqpoint{0.117cm}{1.421cm}}
\pgfpathlineto{\pgfqpoint{0.682cm}{0.671cm}}
\pgfpathlineto{\pgfqpoint{1.246cm}{1.421cm}}
\pgfusepath{stroke}
\end{pgfscope}
\definecolor{eps2pgf_color}{gray}{0}\pgfsetstrokecolor{eps2pgf_color}\pgfsetfillcolor{eps2pgf_color}
\pgfpathmoveto{\pgfqpoint{0.273cm}{1.395cm}}
\pgfpathcurveto{\pgfqpoint{0.273cm}{1.432cm}}{\pgfqpoint{0.259cm}{1.467cm}}{\pgfqpoint{0.233cm}{1.492cm}}
\pgfpathcurveto{\pgfqpoint{0.207cm}{1.518cm}}{\pgfqpoint{0.173cm}{1.532cm}}{\pgfqpoint{0.137cm}{1.532cm}}
\pgfpathcurveto{\pgfqpoint{0.1cm}{1.532cm}}{\pgfqpoint{0.066cm}{1.518cm}}{\pgfqpoint{0.04cm}{1.492cm}}
\pgfpathcurveto{\pgfqpoint{0.014cm}{1.467cm}}{\pgfqpoint{0cm}{1.432cm}}{\pgfqpoint{0cm}{1.395cm}}
\pgfpathcurveto{\pgfqpoint{0cm}{1.359cm}}{\pgfqpoint{0.014cm}{1.324cm}}{\pgfqpoint{0.04cm}{1.299cm}}
\pgfpathcurveto{\pgfqpoint{0.066cm}{1.273cm}}{\pgfqpoint{0.1cm}{1.258cm}}{\pgfqpoint{0.137cm}{1.258cm}}
\pgfpathcurveto{\pgfqpoint{0.173cm}{1.258cm}}{\pgfqpoint{0.207cm}{1.273cm}}{\pgfqpoint{0.233cm}{1.299cm}}
\pgfpathcurveto{\pgfqpoint{0.259cm}{1.324cm}}{\pgfqpoint{0.273cm}{1.359cm}}{\pgfqpoint{0.273cm}{1.395cm}}
\pgfusepath{fill}
\begin{pgfscope}
\pgfsetdash{}{0cm}
\pgfsetlinewidth{0.818mm}
\pgfsetmiterlimit{7.0}
\pgfpathmoveto{\pgfqpoint{0.682cm}{0.671cm}}
\pgfpathlineto{\pgfqpoint{0.679cm}{1.418cm}}
\pgfusepath{stroke}
\end{pgfscope}
\pgfpathmoveto{\pgfqpoint{0.815cm}{1.399cm}}
\pgfpathcurveto{\pgfqpoint{0.815cm}{1.435cm}}{\pgfqpoint{0.801cm}{1.47cm}}{\pgfqpoint{0.775cm}{1.496cm}}
\pgfpathcurveto{\pgfqpoint{0.75cm}{1.521cm}}{\pgfqpoint{0.715cm}{1.536cm}}{\pgfqpoint{0.679cm}{1.536cm}}
\pgfpathcurveto{\pgfqpoint{0.643cm}{1.536cm}}{\pgfqpoint{0.608cm}{1.521cm}}{\pgfqpoint{0.582cm}{1.496cm}}
\pgfpathcurveto{\pgfqpoint{0.557cm}{1.47cm}}{\pgfqpoint{0.542cm}{1.435cm}}{\pgfqpoint{0.542cm}{1.399cm}}
\pgfpathcurveto{\pgfqpoint{0.542cm}{1.363cm}}{\pgfqpoint{0.557cm}{1.328cm}}{\pgfqpoint{0.582cm}{1.302cm}}
\pgfpathcurveto{\pgfqpoint{0.608cm}{1.276cm}}{\pgfqpoint{0.643cm}{1.262cm}}{\pgfqpoint{0.679cm}{1.262cm}}
\pgfpathcurveto{\pgfqpoint{0.715cm}{1.262cm}}{\pgfqpoint{0.75cm}{1.276cm}}{\pgfqpoint{0.775cm}{1.302cm}}
\pgfpathcurveto{\pgfqpoint{0.801cm}{1.328cm}}{\pgfqpoint{0.815cm}{1.363cm}}{\pgfqpoint{0.815cm}{1.399cm}}
\pgfusepath{fill}
\pgfpathmoveto{\pgfqpoint{1.345cm}{1.371cm}}
\pgfpathcurveto{\pgfqpoint{1.345cm}{1.408cm}}{\pgfqpoint{1.331cm}{1.442cm}}{\pgfqpoint{1.305cm}{1.468cm}}
\pgfpathcurveto{\pgfqpoint{1.28cm}{1.494cm}}{\pgfqpoint{1.245cm}{1.508cm}}{\pgfqpoint{1.209cm}{1.508cm}}
\pgfpathcurveto{\pgfqpoint{1.172cm}{1.508cm}}{\pgfqpoint{1.138cm}{1.494cm}}{\pgfqpoint{1.112cm}{1.468cm}}
\pgfpathcurveto{\pgfqpoint{1.087cm}{1.442cm}}{\pgfqpoint{1.072cm}{1.408cm}}{\pgfqpoint{1.072cm}{1.371cm}}
\pgfpathcurveto{\pgfqpoint{1.072cm}{1.335cm}}{\pgfqpoint{1.087cm}{1.3cm}}{\pgfqpoint{1.112cm}{1.274cm}}
\pgfpathcurveto{\pgfqpoint{1.138cm}{1.249cm}}{\pgfqpoint{1.172cm}{1.234cm}}{\pgfqpoint{1.209cm}{1.234cm}}
\pgfpathcurveto{\pgfqpoint{1.245cm}{1.234cm}}{\pgfqpoint{1.28cm}{1.249cm}}{\pgfqpoint{1.305cm}{1.274cm}}
\pgfpathcurveto{\pgfqpoint{1.331cm}{1.3cm}}{\pgfqpoint{1.345cm}{1.335cm}}{\pgfqpoint{1.345cm}{1.371cm}}
\pgfusepath{fill}
\begin{pgfscope}
\pgfsetdash{}{0cm}
\pgfsetlinewidth{0.818mm}
\pgfsetroundcap
\pgfsetmiterlimit{4.0}
\pgfpathmoveto{\pgfqpoint{0.682cm}{0.671cm}}
\pgfpathlineto{\pgfqpoint{0.682cm}{0.042cm}}
\pgfusepath{stroke}
\end{pgfscope}
\end{pgfscope}
\end{pgfscope}
\end{pgfscope}
\end{tikzpicture}}}+\phi+\psi\|_{C^{(\alpha+\kappa)/2}L^{\infty}(\rho^{3+\gamma''})}+\|-X^{\!\resizebox{0.6em}{!}{
\begin{tikzpicture}
\pgfpathmoveto{\pgfqpoint{0cm}{-0.035cm}}
\pgfpathlineto{\pgfqpoint{1.376cm}{-0.035cm}}
\pgfpathlineto{\pgfqpoint{1.376cm}{1.552cm}}
\pgfpathlineto{\pgfqpoint{0cm}{1.552cm}}
\pgfpathclose
\pgfusepath{clip}
\begin{pgfscope}
\begin{pgfscope}
\pgfpathmoveto{\pgfqpoint{0cm}{-0.035cm}}
\pgfpathlineto{\pgfqpoint{1.376cm}{-0.035cm}}
\pgfpathlineto{\pgfqpoint{1.376cm}{1.552cm}}
\pgfpathlineto{\pgfqpoint{0cm}{1.552cm}}
\pgfpathclose
\pgfusepath{clip}
\begin{pgfscope}
\begin{pgfscope}
\pgfsetdash{}{0cm}
\pgfsetlinewidth{0.818mm}
\pgfsetroundcap
\pgfsetroundjoin
\pgfsetmiterlimit{7.0}
\definecolor{eps2pgf_color}{gray}{0}\pgfsetstrokecolor{eps2pgf_color}\pgfsetfillcolor{eps2pgf_color}
\pgfpathmoveto{\pgfqpoint{0.117cm}{1.421cm}}
\pgfpathlineto{\pgfqpoint{0.682cm}{0.671cm}}
\pgfpathlineto{\pgfqpoint{1.246cm}{1.421cm}}
\pgfusepath{stroke}
\end{pgfscope}
\definecolor{eps2pgf_color}{gray}{0}\pgfsetstrokecolor{eps2pgf_color}\pgfsetfillcolor{eps2pgf_color}
\pgfpathmoveto{\pgfqpoint{0.273cm}{1.395cm}}
\pgfpathcurveto{\pgfqpoint{0.273cm}{1.432cm}}{\pgfqpoint{0.259cm}{1.467cm}}{\pgfqpoint{0.233cm}{1.492cm}}
\pgfpathcurveto{\pgfqpoint{0.207cm}{1.518cm}}{\pgfqpoint{0.173cm}{1.532cm}}{\pgfqpoint{0.137cm}{1.532cm}}
\pgfpathcurveto{\pgfqpoint{0.1cm}{1.532cm}}{\pgfqpoint{0.066cm}{1.518cm}}{\pgfqpoint{0.04cm}{1.492cm}}
\pgfpathcurveto{\pgfqpoint{0.014cm}{1.467cm}}{\pgfqpoint{0cm}{1.432cm}}{\pgfqpoint{0cm}{1.395cm}}
\pgfpathcurveto{\pgfqpoint{0cm}{1.359cm}}{\pgfqpoint{0.014cm}{1.324cm}}{\pgfqpoint{0.04cm}{1.299cm}}
\pgfpathcurveto{\pgfqpoint{0.066cm}{1.273cm}}{\pgfqpoint{0.1cm}{1.258cm}}{\pgfqpoint{0.137cm}{1.258cm}}
\pgfpathcurveto{\pgfqpoint{0.173cm}{1.258cm}}{\pgfqpoint{0.207cm}{1.273cm}}{\pgfqpoint{0.233cm}{1.299cm}}
\pgfpathcurveto{\pgfqpoint{0.259cm}{1.324cm}}{\pgfqpoint{0.273cm}{1.359cm}}{\pgfqpoint{0.273cm}{1.395cm}}
\pgfusepath{fill}
\begin{pgfscope}
\pgfsetdash{}{0cm}
\pgfsetlinewidth{0.818mm}
\pgfsetmiterlimit{7.0}
\pgfpathmoveto{\pgfqpoint{0.682cm}{0.671cm}}
\pgfpathlineto{\pgfqpoint{0.679cm}{1.418cm}}
\pgfusepath{stroke}
\end{pgfscope}
\pgfpathmoveto{\pgfqpoint{0.815cm}{1.399cm}}
\pgfpathcurveto{\pgfqpoint{0.815cm}{1.435cm}}{\pgfqpoint{0.801cm}{1.47cm}}{\pgfqpoint{0.775cm}{1.496cm}}
\pgfpathcurveto{\pgfqpoint{0.75cm}{1.521cm}}{\pgfqpoint{0.715cm}{1.536cm}}{\pgfqpoint{0.679cm}{1.536cm}}
\pgfpathcurveto{\pgfqpoint{0.643cm}{1.536cm}}{\pgfqpoint{0.608cm}{1.521cm}}{\pgfqpoint{0.582cm}{1.496cm}}
\pgfpathcurveto{\pgfqpoint{0.557cm}{1.47cm}}{\pgfqpoint{0.542cm}{1.435cm}}{\pgfqpoint{0.542cm}{1.399cm}}
\pgfpathcurveto{\pgfqpoint{0.542cm}{1.363cm}}{\pgfqpoint{0.557cm}{1.328cm}}{\pgfqpoint{0.582cm}{1.302cm}}
\pgfpathcurveto{\pgfqpoint{0.608cm}{1.276cm}}{\pgfqpoint{0.643cm}{1.262cm}}{\pgfqpoint{0.679cm}{1.262cm}}
\pgfpathcurveto{\pgfqpoint{0.715cm}{1.262cm}}{\pgfqpoint{0.75cm}{1.276cm}}{\pgfqpoint{0.775cm}{1.302cm}}
\pgfpathcurveto{\pgfqpoint{0.801cm}{1.328cm}}{\pgfqpoint{0.815cm}{1.363cm}}{\pgfqpoint{0.815cm}{1.399cm}}
\pgfusepath{fill}
\pgfpathmoveto{\pgfqpoint{1.345cm}{1.371cm}}
\pgfpathcurveto{\pgfqpoint{1.345cm}{1.408cm}}{\pgfqpoint{1.331cm}{1.442cm}}{\pgfqpoint{1.305cm}{1.468cm}}
\pgfpathcurveto{\pgfqpoint{1.28cm}{1.494cm}}{\pgfqpoint{1.245cm}{1.508cm}}{\pgfqpoint{1.209cm}{1.508cm}}
\pgfpathcurveto{\pgfqpoint{1.172cm}{1.508cm}}{\pgfqpoint{1.138cm}{1.494cm}}{\pgfqpoint{1.112cm}{1.468cm}}
\pgfpathcurveto{\pgfqpoint{1.087cm}{1.442cm}}{\pgfqpoint{1.072cm}{1.408cm}}{\pgfqpoint{1.072cm}{1.371cm}}
\pgfpathcurveto{\pgfqpoint{1.072cm}{1.335cm}}{\pgfqpoint{1.087cm}{1.3cm}}{\pgfqpoint{1.112cm}{1.274cm}}
\pgfpathcurveto{\pgfqpoint{1.138cm}{1.249cm}}{\pgfqpoint{1.172cm}{1.234cm}}{\pgfqpoint{1.209cm}{1.234cm}}
\pgfpathcurveto{\pgfqpoint{1.245cm}{1.234cm}}{\pgfqpoint{1.28cm}{1.249cm}}{\pgfqpoint{1.305cm}{1.274cm}}
\pgfpathcurveto{\pgfqpoint{1.331cm}{1.3cm}}{\pgfqpoint{1.345cm}{1.335cm}}{\pgfqpoint{1.345cm}{1.371cm}}
\pgfusepath{fill}
\begin{pgfscope}
\pgfsetdash{}{0cm}
\pgfsetlinewidth{0.818mm}
\pgfsetroundcap
\pgfsetmiterlimit{4.0}
\pgfpathmoveto{\pgfqpoint{0.682cm}{0.671cm}}
\pgfpathlineto{\pgfqpoint{0.682cm}{0.042cm}}
\pgfusepath{stroke}
\end{pgfscope}
\end{pgfscope}
\end{pgfscope}
\end{pgfscope}
\end{tikzpicture}}}+\phi+\psi\|_{C_{T}\CC^{\alpha+\kappa}(\rho^{3+\gamma''})}  \Big) \|X^{\!\resizebox{0.6em}{!}{
\begin{tikzpicture}
\pgfpathmoveto{\pgfqpoint{0cm}{0cm}}
\pgfpathlineto{\pgfqpoint{1.376cm}{0cm}}
\pgfpathlineto{\pgfqpoint{1.376cm}{1.588cm}}
\pgfpathlineto{\pgfqpoint{0cm}{1.588cm}}
\pgfpathclose
\pgfusepath{clip}
\begin{pgfscope}
\begin{pgfscope}
\pgfpathmoveto{\pgfqpoint{0cm}{0cm}}
\pgfpathlineto{\pgfqpoint{1.376cm}{0cm}}
\pgfpathlineto{\pgfqpoint{1.376cm}{1.588cm}}
\pgfpathlineto{\pgfqpoint{0cm}{1.588cm}}
\pgfpathclose
\pgfusepath{clip}
\begin{pgfscope}
\begin{pgfscope}
\definecolor{eps2pgf_color}{gray}{0.976471}\pgfsetstrokecolor{eps2pgf_color}\pgfsetfillcolor{eps2pgf_color}
\pgfpathmoveto{\pgfqpoint{0cm}{0cm}}
\pgfpathlineto{\pgfqpoint{1.376cm}{0cm}}
\pgfpathlineto{\pgfqpoint{1.376cm}{1.588cm}}
\pgfpathlineto{\pgfqpoint{0cm}{1.588cm}}
\pgfpathclose
\pgfusepath{fill}
\end{pgfscope}
\begin{pgfscope}
\pgfsetdash{}{0cm}
\pgfsetlinewidth{0.818mm}
\pgfsetroundcap
\pgfsetroundjoin
\pgfsetmiterlimit{7.0}
\definecolor{eps2pgf_color}{gray}{0}\pgfsetstrokecolor{eps2pgf_color}\pgfsetfillcolor{eps2pgf_color}
\pgfpathmoveto{\pgfqpoint{0.117cm}{1.476cm}}
\pgfpathlineto{\pgfqpoint{0.682cm}{0.726cm}}
\pgfpathlineto{\pgfqpoint{1.246cm}{1.476cm}}
\pgfusepath{stroke}
\end{pgfscope}
\definecolor{eps2pgf_color}{gray}{0}\pgfsetstrokecolor{eps2pgf_color}\pgfsetfillcolor{eps2pgf_color}
\pgfpathmoveto{\pgfqpoint{0.273cm}{1.451cm}}
\pgfpathcurveto{\pgfqpoint{0.273cm}{1.487cm}}{\pgfqpoint{0.259cm}{1.522cm}}{\pgfqpoint{0.233cm}{1.547cm}}
\pgfpathcurveto{\pgfqpoint{0.207cm}{1.573cm}}{\pgfqpoint{0.173cm}{1.588cm}}{\pgfqpoint{0.137cm}{1.588cm}}
\pgfpathcurveto{\pgfqpoint{0.1cm}{1.588cm}}{\pgfqpoint{0.066cm}{1.573cm}}{\pgfqpoint{0.04cm}{1.547cm}}
\pgfpathcurveto{\pgfqpoint{0.014cm}{1.522cm}}{\pgfqpoint{0cm}{1.487cm}}{\pgfqpoint{0cm}{1.451cm}}
\pgfpathcurveto{\pgfqpoint{0cm}{1.414cm}}{\pgfqpoint{0.014cm}{1.379cm}}{\pgfqpoint{0.04cm}{1.354cm}}
\pgfpathcurveto{\pgfqpoint{0.066cm}{1.328cm}}{\pgfqpoint{0.1cm}{1.314cm}}{\pgfqpoint{0.137cm}{1.314cm}}
\pgfpathcurveto{\pgfqpoint{0.173cm}{1.314cm}}{\pgfqpoint{0.207cm}{1.328cm}}{\pgfqpoint{0.233cm}{1.354cm}}
\pgfpathcurveto{\pgfqpoint{0.259cm}{1.379cm}}{\pgfqpoint{0.273cm}{1.414cm}}{\pgfqpoint{0.273cm}{1.451cm}}
\pgfusepath{fill}
\pgfpathmoveto{\pgfqpoint{1.345cm}{1.426cm}}
\pgfpathcurveto{\pgfqpoint{1.345cm}{1.463cm}}{\pgfqpoint{1.331cm}{1.497cm}}{\pgfqpoint{1.305cm}{1.523cm}}
\pgfpathcurveto{\pgfqpoint{1.28cm}{1.549cm}}{\pgfqpoint{1.245cm}{1.563cm}}{\pgfqpoint{1.209cm}{1.563cm}}
\pgfpathcurveto{\pgfqpoint{1.172cm}{1.563cm}}{\pgfqpoint{1.138cm}{1.549cm}}{\pgfqpoint{1.112cm}{1.523cm}}
\pgfpathcurveto{\pgfqpoint{1.087cm}{1.497cm}}{\pgfqpoint{1.072cm}{1.463cm}}{\pgfqpoint{1.072cm}{1.426cm}}
\pgfpathcurveto{\pgfqpoint{1.072cm}{1.39cm}}{\pgfqpoint{1.087cm}{1.355cm}}{\pgfqpoint{1.112cm}{1.329cm}}
\pgfpathcurveto{\pgfqpoint{1.138cm}{1.304cm}}{\pgfqpoint{1.172cm}{1.289cm}}{\pgfqpoint{1.209cm}{1.289cm}}
\pgfpathcurveto{\pgfqpoint{1.245cm}{1.289cm}}{\pgfqpoint{1.28cm}{1.304cm}}{\pgfqpoint{1.305cm}{1.329cm}}
\pgfpathcurveto{\pgfqpoint{1.331cm}{1.355cm}}{\pgfqpoint{1.345cm}{1.39cm}}{\pgfqpoint{1.345cm}{1.426cm}}
\pgfusepath{fill}
\begin{pgfscope}
\pgfsetdash{}{0cm}
\pgfsetlinewidth{0.818mm}
\pgfsetroundcap
\pgfsetmiterlimit{4.0}
\pgfpathmoveto{\pgfqpoint{0.682cm}{0.726cm}}
\pgfpathlineto{\pgfqpoint{0.682cm}{0.097cm}}
\pgfusepath{stroke}
\end{pgfscope}
\end{pgfscope}
\end{pgfscope}
\end{pgfscope}
\end{tikzpicture}}}\|_{C\CC^{1-\kappa}(\rho^{\sigma})}\\
&\quad \lesssim 1 +  \|\phi+\psi\|_{C^{( \alpha+\kappa)/2}L^{\infty}(\rho^{ 3+\gamma''})}+\|\phi+\psi\|_{C\CC^{\frac 12 +\alpha}(\rho^{\frac 32+\alpha})},
\end{align*}
where $0<\gamma''<\gamma'$ and we chose $\alpha$ and $\kappa$ sufficiently small such that  $\alpha+\kappa<\delta$ for $\delta$ from Theorem \ref{thm:renorm43}.
All the other terms in $\Theta$ can be estimated pointwise in time by the same approach as in Section \ref{ssec:theta}.
Therefore, it only remains to bound the time regularity of $\phi+\psi$. Choosing $\gamma''$ sufficiently large and using Lemma \ref{lemma:interp2}, we obtain for $\lambda\in (0,1)$
\begin{align*}
\|\phi+\psi\|_{C^{( \alpha+\kappa)/2}L^{\infty}(\rho^{ 3+\gamma''})}&\lesssim \|\phi+\psi\|_{L^{\8}L^{\infty}(\rho)}+\lambda\|\phi+\psi\|_{C^{(\frac 12+ \alpha)/2}L^{\infty}(\rho^{ 3+\gamma})}\\
&\lesssim \|\phi+\psi\|_{L^{\8}L^{\infty}(\rho)}+\|\phi\|_{C^{(\frac 12+ \alpha)/2}L^{\infty}(\rho^{ \frac 32+\alpha})}+\lambda\|\psi\|_{C^{1}L^{\infty}(\rho^{  3+\gamma})}.
\end{align*}
Observe that the small parameter $\lambda$ is only needed in order to absorb the $C^{1}$-norm of $\psi$ into the left hand side.
And the same bound holds true for $\phi+\psi\in C^{( \alpha+\kappa)/2}L^{\infty}(\rho^{ 3+\gamma'})$ needed to estimate the additional term in $\rmb{\tilde\Psi}$.
This is sufficient in order to obtain the desired uniform bounds for $\phi,\vartheta,\psi$.

To be more precise, as in the case of the parabolic $\Phi^{4}$ model in dimension 2, we show existence via a smooth approximation and compactness. To this end, let $\xi_{\varepsilon}$ be a smooth and periodic approximation  of the driving space-time white noise $\xi$, defined on the torus of size $\frac 1\varepsilon$.
If  $\varphi_{\varepsilon,0}$ is a smooth approximation  of the initial condition $\varphi_{0}$, then according to Proposition \ref{prop:42d}, for every $\varepsilon\in (0,1)$ there exists $\varphi_{\varepsilon}\in C^{\infty}([0,T]\times \mathbb{T}^{3}_{1/\varepsilon})$ which is the unique classical solution to
\begin{equation*}
\LL \varphi_{\varepsilon} + \varphi_{\varepsilon}^3+ ( - 3 a_{\varepsilon} +3 b_{\varepsilon})\varphi_{\varepsilon} - \xi_{\varepsilon} = 0,\qquad\varphi(0)=\varphi_{\varepsilon,0}.
\end{equation*}
Now, we proceed with the same decomposition $\varphi_{\varepsilon}=X_{\varepsilon} - X_{\varepsilon}^{\!\resizebox{0.6em}{!}{
\begin{tikzpicture}
\pgfpathmoveto{\pgfqpoint{0cm}{-0.035cm}}
\pgfpathlineto{\pgfqpoint{1.376cm}{-0.035cm}}
\pgfpathlineto{\pgfqpoint{1.376cm}{1.552cm}}
\pgfpathlineto{\pgfqpoint{0cm}{1.552cm}}
\pgfpathclose
\pgfusepath{clip}
\begin{pgfscope}
\begin{pgfscope}
\pgfpathmoveto{\pgfqpoint{0cm}{-0.035cm}}
\pgfpathlineto{\pgfqpoint{1.376cm}{-0.035cm}}
\pgfpathlineto{\pgfqpoint{1.376cm}{1.552cm}}
\pgfpathlineto{\pgfqpoint{0cm}{1.552cm}}
\pgfpathclose
\pgfusepath{clip}
\begin{pgfscope}
\begin{pgfscope}
\pgfsetdash{}{0cm}
\pgfsetlinewidth{0.818mm}
\pgfsetroundcap
\pgfsetroundjoin
\pgfsetmiterlimit{7.0}
\definecolor{eps2pgf_color}{gray}{0}\pgfsetstrokecolor{eps2pgf_color}\pgfsetfillcolor{eps2pgf_color}
\pgfpathmoveto{\pgfqpoint{0.117cm}{1.421cm}}
\pgfpathlineto{\pgfqpoint{0.682cm}{0.671cm}}
\pgfpathlineto{\pgfqpoint{1.246cm}{1.421cm}}
\pgfusepath{stroke}
\end{pgfscope}
\definecolor{eps2pgf_color}{gray}{0}\pgfsetstrokecolor{eps2pgf_color}\pgfsetfillcolor{eps2pgf_color}
\pgfpathmoveto{\pgfqpoint{0.273cm}{1.395cm}}
\pgfpathcurveto{\pgfqpoint{0.273cm}{1.432cm}}{\pgfqpoint{0.259cm}{1.467cm}}{\pgfqpoint{0.233cm}{1.492cm}}
\pgfpathcurveto{\pgfqpoint{0.207cm}{1.518cm}}{\pgfqpoint{0.173cm}{1.532cm}}{\pgfqpoint{0.137cm}{1.532cm}}
\pgfpathcurveto{\pgfqpoint{0.1cm}{1.532cm}}{\pgfqpoint{0.066cm}{1.518cm}}{\pgfqpoint{0.04cm}{1.492cm}}
\pgfpathcurveto{\pgfqpoint{0.014cm}{1.467cm}}{\pgfqpoint{0cm}{1.432cm}}{\pgfqpoint{0cm}{1.395cm}}
\pgfpathcurveto{\pgfqpoint{0cm}{1.359cm}}{\pgfqpoint{0.014cm}{1.324cm}}{\pgfqpoint{0.04cm}{1.299cm}}
\pgfpathcurveto{\pgfqpoint{0.066cm}{1.273cm}}{\pgfqpoint{0.1cm}{1.258cm}}{\pgfqpoint{0.137cm}{1.258cm}}
\pgfpathcurveto{\pgfqpoint{0.173cm}{1.258cm}}{\pgfqpoint{0.207cm}{1.273cm}}{\pgfqpoint{0.233cm}{1.299cm}}
\pgfpathcurveto{\pgfqpoint{0.259cm}{1.324cm}}{\pgfqpoint{0.273cm}{1.359cm}}{\pgfqpoint{0.273cm}{1.395cm}}
\pgfusepath{fill}
\begin{pgfscope}
\pgfsetdash{}{0cm}
\pgfsetlinewidth{0.818mm}
\pgfsetmiterlimit{7.0}
\pgfpathmoveto{\pgfqpoint{0.682cm}{0.671cm}}
\pgfpathlineto{\pgfqpoint{0.679cm}{1.418cm}}
\pgfusepath{stroke}
\end{pgfscope}
\pgfpathmoveto{\pgfqpoint{0.815cm}{1.399cm}}
\pgfpathcurveto{\pgfqpoint{0.815cm}{1.435cm}}{\pgfqpoint{0.801cm}{1.47cm}}{\pgfqpoint{0.775cm}{1.496cm}}
\pgfpathcurveto{\pgfqpoint{0.75cm}{1.521cm}}{\pgfqpoint{0.715cm}{1.536cm}}{\pgfqpoint{0.679cm}{1.536cm}}
\pgfpathcurveto{\pgfqpoint{0.643cm}{1.536cm}}{\pgfqpoint{0.608cm}{1.521cm}}{\pgfqpoint{0.582cm}{1.496cm}}
\pgfpathcurveto{\pgfqpoint{0.557cm}{1.47cm}}{\pgfqpoint{0.542cm}{1.435cm}}{\pgfqpoint{0.542cm}{1.399cm}}
\pgfpathcurveto{\pgfqpoint{0.542cm}{1.363cm}}{\pgfqpoint{0.557cm}{1.328cm}}{\pgfqpoint{0.582cm}{1.302cm}}
\pgfpathcurveto{\pgfqpoint{0.608cm}{1.276cm}}{\pgfqpoint{0.643cm}{1.262cm}}{\pgfqpoint{0.679cm}{1.262cm}}
\pgfpathcurveto{\pgfqpoint{0.715cm}{1.262cm}}{\pgfqpoint{0.75cm}{1.276cm}}{\pgfqpoint{0.775cm}{1.302cm}}
\pgfpathcurveto{\pgfqpoint{0.801cm}{1.328cm}}{\pgfqpoint{0.815cm}{1.363cm}}{\pgfqpoint{0.815cm}{1.399cm}}
\pgfusepath{fill}
\pgfpathmoveto{\pgfqpoint{1.345cm}{1.371cm}}
\pgfpathcurveto{\pgfqpoint{1.345cm}{1.408cm}}{\pgfqpoint{1.331cm}{1.442cm}}{\pgfqpoint{1.305cm}{1.468cm}}
\pgfpathcurveto{\pgfqpoint{1.28cm}{1.494cm}}{\pgfqpoint{1.245cm}{1.508cm}}{\pgfqpoint{1.209cm}{1.508cm}}
\pgfpathcurveto{\pgfqpoint{1.172cm}{1.508cm}}{\pgfqpoint{1.138cm}{1.494cm}}{\pgfqpoint{1.112cm}{1.468cm}}
\pgfpathcurveto{\pgfqpoint{1.087cm}{1.442cm}}{\pgfqpoint{1.072cm}{1.408cm}}{\pgfqpoint{1.072cm}{1.371cm}}
\pgfpathcurveto{\pgfqpoint{1.072cm}{1.335cm}}{\pgfqpoint{1.087cm}{1.3cm}}{\pgfqpoint{1.112cm}{1.274cm}}
\pgfpathcurveto{\pgfqpoint{1.138cm}{1.249cm}}{\pgfqpoint{1.172cm}{1.234cm}}{\pgfqpoint{1.209cm}{1.234cm}}
\pgfpathcurveto{\pgfqpoint{1.245cm}{1.234cm}}{\pgfqpoint{1.28cm}{1.249cm}}{\pgfqpoint{1.305cm}{1.274cm}}
\pgfpathcurveto{\pgfqpoint{1.331cm}{1.3cm}}{\pgfqpoint{1.345cm}{1.335cm}}{\pgfqpoint{1.345cm}{1.371cm}}
\pgfusepath{fill}
\begin{pgfscope}
\pgfsetdash{}{0cm}
\pgfsetlinewidth{0.818mm}
\pgfsetroundcap
\pgfsetmiterlimit{4.0}
\pgfpathmoveto{\pgfqpoint{0.682cm}{0.671cm}}
\pgfpathlineto{\pgfqpoint{0.682cm}{0.042cm}}
\pgfusepath{stroke}
\end{pgfscope}
\end{pgfscope}
\end{pgfscope}
\end{pgfscope}
\end{tikzpicture}}} +\phi_{\varepsilon}+\psi_{\varepsilon}$ as above, only starting from the mollified noise $\xi_{\varepsilon}$ instead of  $\xi$. 
Next, in view of the bounds from Theorem \ref{thm:renorm43}, the a priori estimates from Sections~\ref{ssec:phi1}, \ref{ssec:phi2}, \ref{ssec:theta}, \ref{ssec:psi1}, \ref{ssec:psi2} apply mutatis mutandis  (using  Lemmas~\ref{lemma:interp2}, \ref{lem:local2}, \ref{lemma:sch}, \ref{lemma:schauder-par}, \ref{lemma:apriori-parabolic}), with only slight modification due to the necessary time regularity needed for the additional commutators in $\rmb{\tilde\Psi}$ and $\Theta$.

To summarize, we deduce that the following bounds hold true uniformly in $\varepsilon\in (0,1)$
\begin{align}\label{eq:un43}
\begin{aligned}
\|\phi_{\varepsilon}\|_{C\CC^{\alpha}(\rho)}+\|\phi_{\varepsilon}\|_{C^{\alpha/2}L^{\infty}(\rho)}+\|\phi_{\varepsilon}\|_{C\CC^{\frac 12+\alpha}(\rho^{\frac 32+\alpha})}+\|\phi_{\varepsilon}\|_{C^{(\frac 12+\alpha)/2}L^{\infty}(\rho^{\frac 32+\alpha})}&\lesssim 1,\\
\|\vartheta_{\varepsilon}\|_{C\CC^{1+\alpha}(\rho^{3+\gamma'})}+\|\vartheta_{\varepsilon}\|_{C^{(1+\alpha)/2}L^{\infty}(\rho^{3+\gamma'})}&\lesssim 1\\
\|\psi_{\varepsilon}\|_{C\CC^{2+\gamma}(\rho^{3+\gamma})}+\|\psi_{\varepsilon}\|_{C^{1}L^{\infty}(\rho^{3+\gamma})}+\|\psi_{\varepsilon}\|_{L^{\infty}L^{\infty}(\rho)}&\lesssim 1.
\end{aligned}
\end{align}

Consequently, we are able  to pass to the limit.

\begin{theorem}\label{thm:ex3d}
Let $\kappa, \alpha\in (0,1)$ be chosen sufficiently small and let $\gamma=\alpha-\kappa>0$ and $\gamma'\in(0,\gamma)$ sufficiently large. If $\varphi_{0}\in \CC^{1+\alpha}(\rho^{3+\gamma'}_{0})$ then there exist
\begin{align*}
\phi&\in C\CC^{\alpha}(\rho)\cap C^{\alpha/2}L^{\infty}(\rho)\cap C\CC^{\frac 12+\alpha}(\rho^{\frac 32+\alpha})\cap C^{(\frac 12+\alpha)/2}L^{\infty}(\rho^{\frac 32+\alpha}),\\
\vartheta&\in C\CC^{1+\alpha}(\rho^{3+\gamma'}) \cap C^{(1+\alpha)/2}L^{\infty}(\rho^{3+\gamma'}),\\
\psi&\in C\CC^{2+\gamma}(\rho^{3+\gamma})\cap C^{1}L^{\infty}(\rho^{3+\gamma})\cap L^{\infty}L^{\infty}(\rho),
\end{align*}
which is a solution to  \eqref{eq:two43}, \eqref{eq:th43}.
\end{theorem}

\begin{proof}
The proof follows the lines of Theorem \ref{thm:ex42}. Based on the uniform bounds from \eqref{eq:un43} we obtain compactness of the sequence of approximate solutions $(\phi_{\varepsilon},\vartheta_{\varepsilon},\psi_{\varepsilon})$ in a slightly worse space. In view of Theorem \ref{thm:renorm43}, this allows us to pass to the limit in the approximate version of \eqref{eq:two43}, \eqref{eq:th43}. Finally, we obtain that the limit solutions belong to the spaces where the uniform bounds hold.
\end{proof}

\section{Uniqueness for the parabolic models}
\label{s:uniq}

This  section is concerned with uniqueness to the parabolic
$\Phi^4_d$ model for $d = 2, 3$.

\begin{theorem}
  \label{thm:uniq3d}
The parabolic $\Phi^{4}$ model \eqref{eq:phi4p} in dimension 2 and 3 has a unique solution: Let $(\phi,\psi)$, $(\tilde\phi,\tilde\psi)$ be two solutions in the sense of Theorem \ref{thm:ex42} and Theorem \ref{thm:ex3d}, respectively, starting from an initial condition $\phi_{0}+\psi_{0}=\tilde\phi_{0}+\tilde\psi_{0}=\varphi_{0}$. Then $\phi+\psi=\tilde\phi+\tilde\psi$.
\end{theorem}

In what follows we present the proof of this result in the more involved setting of $d=3$. The two dimensional case follows the same pattern but is significantly easier and  the details are left  to the reader. Since the cubic term does not seem to be helpful for uniqueness, namely, when we study the equation for a difference of two solutions, it is necessary to find another mechanism which could handle the loss of weight in the terms of lower order. This issue can  be easily seen on the model equation
$$
\LL v = v f, \quad v(0)=0,
$$
which is the form the equation for the difference takes. Intuitively, if $f$ can only be bounded in a weighted space and accordingly also $v$ is bounded in a weighted space, then the product $vf$ can only be bounded when multiplied by the product of the two weights. Hence the right hand side requires higher weight than the left hand side which causes difficulties in closing the estimates.

We overcome this problem by introducing an exponential weight of the form $\rho(x)\pi(t,x):=\langle x \rangle^{- a} e^{- t \langle x
\rangle^b}$ for $a \in \mathbbm{R}$ and $b \in (0, 1)$. As usual $t\in [0,\8)$ denotes the time variable. However, this is not an admissible weight in the sense of Section \ref{ssec:besov} and consequently the definition of the associated weighted Besov spaces requires a different approach, either employing ultra-distributions (see \cite{MR891189}) or Gevrey classes (see \cite{MR1249275}). In Section \ref{ss:b} we recall the basic ideas  based on Gevrey classes following the detailed presentation of \cite{MW17}, where we also refer the reader for further
details. Note that exponential weights have already been employed in \cite{MR3358965, MR3779690}.

With suitable weighted Besov spaces at hand, we employ the classical $L^{2}$-energy technique. First, and similarly to the previous sections, we decompose the equation for the difference of two solutions into its regular and irregular components. Then we test both equations by a suitable test function, which corresponds to  the chain rule for certain Besov norm in the $L^{2}$-scale. This way we obtain a control of the $B_{2,2}^{\beta}(\pi_{t})$-norm of the regular component and the $B_{2,2}^{-\beta}(\pi_{t})$-norm of the irregular one, for some $\beta\in (0,1)$. The advantage of the exponential weight $\pi$ (which depends on time) originates in the form of its time derivative. More precisely, this gives a good term on the left hand side with weight of the form $\pi\rho^{-2b}$, that is, explosive at infinity in the space variable. This term is essential in order to control all the terms on the right hand side.

\subsection{Besov spaces with exponential weights}
\label{ss:b}

For the proof of Theorem \ref{thm:uniq3d}, we will employ weighted Besov
spaces with weights of the form $\rho(x)\pi(t,x):=\langle x \rangle^{- a} e^{- t \langle x
\rangle^b}$ for $a \in \mathbbm{R}$ and $b \in (0, 1)$ and $t\in [0,\8)$, which stands for the time variable. In order to compensate for the exponential growth, the definition of the
corresponding Besov spaces relies on the so-called Gevrey classes rather than
on Schwartz functions. Since multiplication by the polynomial weight $\langle
x \rangle^{- a}$ only introduces a logarithmic correction, namely, $\langle x
\rangle^{- a} e^{- t \langle x \rangle^b} = e^{- t \langle x \rangle^b - a
\log \langle x \rangle}$, we may work with the same Gevrey class
$\mathcal{G}^{\theta}$ of index $\theta \in (1, 1 / b)$ as for the case of only exponential weight $e^{- t \langle x
\rangle^b}$.
Consequently, the  results of \cite[Section~2]{MW17} remain valid and given $T>0$ the corresponding bounds are uniform over all $t\in [0,T]$.

Next,  we define the weighted Besov spaces (based on a partition of
unity from $\mathcal{G}^{\theta}$), as the completion of $C^{\infty}_c$ with
respect to the norm
\[ \| f \|_{B^{\alpha}_{p, q} (\pi_{t}\rho)} \assign \left( \sum_{k \geqslant - 1}
   (2^{\alpha k} \| \Delta_k f \|_{L^p (\pi_{t}\rho)})^q \right)^{\frac{1}{q}} =
   \left( \sum_{k \geqslant - 1} (2^{\alpha k} \| \pi_{t}\rho \Delta_k f \|_{L^p})^q
   \right)^{\frac{1}{q}}, \]
where $\pi_{t}(\cdot)=\pi(t,\cdot)$.
Note that unlike \cite{MW17} we pull the weight inside the $L^p$-norm, which is consistent with our definition of weighted Besov spaces in Section \ref{ssec:besov}. The corresponding results of \cite[Section 3.3]{MW17} (with straightforward modifications due to the weights) remain valid.
More precisely, the following paraproduct estimates will be used in the sequel.

\begin{lemma}
  \label{lem:par1}Let $\kappa \in [0, 1]$, $\beta \in \mathbbm{R}$ and $\delta > 0$. Then it
  holds uniformly in $t\geq 0$
  \[ \| f \prec g \|_{B_{2, 2}^{\beta} (\pi_{t} \rho)} \lesssim \| f \|_{L^2
     (\pi_{t})} \| g \|_{\CC^{\beta + \delta} (\rho)}\wedge \| f
     \|_{L^{\infty} (\rho)} \| g \|_{B^{\beta}_{2, 2} (\pi_{t})}, \]
  and if $\alpha < 0$ then uniformly over $t\geq 0$
    \[ \| f \prec g \|_{B_{2, 2}^{\alpha + \beta} (\pi_{t} \rho)} \lesssim \| f
     \|_{B^{\alpha}_{2, 2} (\pi_{t})} \| g \|_{\CC^{\beta+\delta} (\rho)} \wedge\| f
     \|_{\CC^{\alpha} (\rho)} \| g \|_{B^{\beta}_{2, 2} (\pi_{t})} . \]
  If $\alpha, \beta \in \mathbbm{R}$ such that $\alpha + \beta > 0$ then it
  holds uniformly in $t\geq 0$
  \[ \| f \circ g \|_{B_{2, 2}^{\alpha + \beta} (\pi_{t} \rho^{})} \lesssim \| f
     \|_{\CC^{\alpha} (\rho)} \| g \|_{B^{\beta}_{2, 2} (\pi_{t})} . \]
\end{lemma}

\begin{proof}
Let  $0<\gamma<\delta$. As a consequence of  \cite[Theorem 3.17]{MW17} and embeddings of Besov spaces, we have
  \[ \| f \prec g \|_{B_{2, 2}^{\beta} (\pi_{t} \rho)} \lesssim \| f \|_{B^{-
     \gamma}_{2, \infty} (\pi_{t})} \| g \|_{B^{\beta+\gamma}_{\infty, 2} (\rho)}
     \lesssim \| f \|_{L^2 (\pi_{t})} \| g \|_{B^{\beta + \delta}_{\infty, \infty}
     (\rho)} . \]
     So the first bound follows and similarly we obtain the third bound.
     The remaining bounds follow directly from  \cite[Theorem 3.17]{MW17}.
\end{proof}

Similarly to Lemma \ref{lem:com} we obtain the following result, whose  proof is a straightforward modification of \cite[Lemma 2.4]{GIP}.
\begin{lemma}
  \label{lem:com1}
  Let $\alpha \in (0, 1)$ and $\beta,
  \gamma \in \mathbbm{R}$ such that $\alpha + \beta + \gamma > 0$ and $\beta +
  \gamma < 0$. Then for every $\delta > 0$ it holds uniformly in $t\ge 0$
  \[ \| \tmop{com} (f, g, h) \|_{B^{\alpha + \beta + \gamma}_{2, 2} (\pi_{t}
     \rho_{1}\rho_{2})} \lesssim \| f \|_{B^{\alpha}_{2, 2} (\pi_{t})} \| g \|_{\CC^{\beta}
     (\rho_{1})} \| h \|_{\CC^{\gamma + \delta} (\rho_{2})} . \]
\end{lemma}

\subsection{Proof of Theorem \ref{thm:uniq3d}}

\begin{proof}
We prove the result for the case $d=3$.  Recall that if $\varphi$ is a solution to \eqref{eq:phi4p}  in the sense of  Theorem \ref{thm:ex3d} then $\varphi = X - X^{\!\resizebox{0.6em}{!}{
\begin{tikzpicture}
\pgfpathmoveto{\pgfqpoint{0cm}{-0.035cm}}
\pgfpathlineto{\pgfqpoint{1.376cm}{-0.035cm}}
\pgfpathlineto{\pgfqpoint{1.376cm}{1.552cm}}
\pgfpathlineto{\pgfqpoint{0cm}{1.552cm}}
\pgfpathclose
\pgfusepath{clip}
\begin{pgfscope}
\begin{pgfscope}
\pgfpathmoveto{\pgfqpoint{0cm}{-0.035cm}}
\pgfpathlineto{\pgfqpoint{1.376cm}{-0.035cm}}
\pgfpathlineto{\pgfqpoint{1.376cm}{1.552cm}}
\pgfpathlineto{\pgfqpoint{0cm}{1.552cm}}
\pgfpathclose
\pgfusepath{clip}
\begin{pgfscope}
\begin{pgfscope}
\pgfsetdash{}{0cm}
\pgfsetlinewidth{0.818mm}
\pgfsetroundcap
\pgfsetroundjoin
\pgfsetmiterlimit{7.0}
\definecolor{eps2pgf_color}{gray}{0}\pgfsetstrokecolor{eps2pgf_color}\pgfsetfillcolor{eps2pgf_color}
\pgfpathmoveto{\pgfqpoint{0.117cm}{1.421cm}}
\pgfpathlineto{\pgfqpoint{0.682cm}{0.671cm}}
\pgfpathlineto{\pgfqpoint{1.246cm}{1.421cm}}
\pgfusepath{stroke}
\end{pgfscope}
\definecolor{eps2pgf_color}{gray}{0}\pgfsetstrokecolor{eps2pgf_color}\pgfsetfillcolor{eps2pgf_color}
\pgfpathmoveto{\pgfqpoint{0.273cm}{1.395cm}}
\pgfpathcurveto{\pgfqpoint{0.273cm}{1.432cm}}{\pgfqpoint{0.259cm}{1.467cm}}{\pgfqpoint{0.233cm}{1.492cm}}
\pgfpathcurveto{\pgfqpoint{0.207cm}{1.518cm}}{\pgfqpoint{0.173cm}{1.532cm}}{\pgfqpoint{0.137cm}{1.532cm}}
\pgfpathcurveto{\pgfqpoint{0.1cm}{1.532cm}}{\pgfqpoint{0.066cm}{1.518cm}}{\pgfqpoint{0.04cm}{1.492cm}}
\pgfpathcurveto{\pgfqpoint{0.014cm}{1.467cm}}{\pgfqpoint{0cm}{1.432cm}}{\pgfqpoint{0cm}{1.395cm}}
\pgfpathcurveto{\pgfqpoint{0cm}{1.359cm}}{\pgfqpoint{0.014cm}{1.324cm}}{\pgfqpoint{0.04cm}{1.299cm}}
\pgfpathcurveto{\pgfqpoint{0.066cm}{1.273cm}}{\pgfqpoint{0.1cm}{1.258cm}}{\pgfqpoint{0.137cm}{1.258cm}}
\pgfpathcurveto{\pgfqpoint{0.173cm}{1.258cm}}{\pgfqpoint{0.207cm}{1.273cm}}{\pgfqpoint{0.233cm}{1.299cm}}
\pgfpathcurveto{\pgfqpoint{0.259cm}{1.324cm}}{\pgfqpoint{0.273cm}{1.359cm}}{\pgfqpoint{0.273cm}{1.395cm}}
\pgfusepath{fill}
\begin{pgfscope}
\pgfsetdash{}{0cm}
\pgfsetlinewidth{0.818mm}
\pgfsetmiterlimit{7.0}
\pgfpathmoveto{\pgfqpoint{0.682cm}{0.671cm}}
\pgfpathlineto{\pgfqpoint{0.679cm}{1.418cm}}
\pgfusepath{stroke}
\end{pgfscope}
\pgfpathmoveto{\pgfqpoint{0.815cm}{1.399cm}}
\pgfpathcurveto{\pgfqpoint{0.815cm}{1.435cm}}{\pgfqpoint{0.801cm}{1.47cm}}{\pgfqpoint{0.775cm}{1.496cm}}
\pgfpathcurveto{\pgfqpoint{0.75cm}{1.521cm}}{\pgfqpoint{0.715cm}{1.536cm}}{\pgfqpoint{0.679cm}{1.536cm}}
\pgfpathcurveto{\pgfqpoint{0.643cm}{1.536cm}}{\pgfqpoint{0.608cm}{1.521cm}}{\pgfqpoint{0.582cm}{1.496cm}}
\pgfpathcurveto{\pgfqpoint{0.557cm}{1.47cm}}{\pgfqpoint{0.542cm}{1.435cm}}{\pgfqpoint{0.542cm}{1.399cm}}
\pgfpathcurveto{\pgfqpoint{0.542cm}{1.363cm}}{\pgfqpoint{0.557cm}{1.328cm}}{\pgfqpoint{0.582cm}{1.302cm}}
\pgfpathcurveto{\pgfqpoint{0.608cm}{1.276cm}}{\pgfqpoint{0.643cm}{1.262cm}}{\pgfqpoint{0.679cm}{1.262cm}}
\pgfpathcurveto{\pgfqpoint{0.715cm}{1.262cm}}{\pgfqpoint{0.75cm}{1.276cm}}{\pgfqpoint{0.775cm}{1.302cm}}
\pgfpathcurveto{\pgfqpoint{0.801cm}{1.328cm}}{\pgfqpoint{0.815cm}{1.363cm}}{\pgfqpoint{0.815cm}{1.399cm}}
\pgfusepath{fill}
\pgfpathmoveto{\pgfqpoint{1.345cm}{1.371cm}}
\pgfpathcurveto{\pgfqpoint{1.345cm}{1.408cm}}{\pgfqpoint{1.331cm}{1.442cm}}{\pgfqpoint{1.305cm}{1.468cm}}
\pgfpathcurveto{\pgfqpoint{1.28cm}{1.494cm}}{\pgfqpoint{1.245cm}{1.508cm}}{\pgfqpoint{1.209cm}{1.508cm}}
\pgfpathcurveto{\pgfqpoint{1.172cm}{1.508cm}}{\pgfqpoint{1.138cm}{1.494cm}}{\pgfqpoint{1.112cm}{1.468cm}}
\pgfpathcurveto{\pgfqpoint{1.087cm}{1.442cm}}{\pgfqpoint{1.072cm}{1.408cm}}{\pgfqpoint{1.072cm}{1.371cm}}
\pgfpathcurveto{\pgfqpoint{1.072cm}{1.335cm}}{\pgfqpoint{1.087cm}{1.3cm}}{\pgfqpoint{1.112cm}{1.274cm}}
\pgfpathcurveto{\pgfqpoint{1.138cm}{1.249cm}}{\pgfqpoint{1.172cm}{1.234cm}}{\pgfqpoint{1.209cm}{1.234cm}}
\pgfpathcurveto{\pgfqpoint{1.245cm}{1.234cm}}{\pgfqpoint{1.28cm}{1.249cm}}{\pgfqpoint{1.305cm}{1.274cm}}
\pgfpathcurveto{\pgfqpoint{1.331cm}{1.3cm}}{\pgfqpoint{1.345cm}{1.335cm}}{\pgfqpoint{1.345cm}{1.371cm}}
\pgfusepath{fill}
\begin{pgfscope}
\pgfsetdash{}{0cm}
\pgfsetlinewidth{0.818mm}
\pgfsetroundcap
\pgfsetmiterlimit{4.0}
\pgfpathmoveto{\pgfqpoint{0.682cm}{0.671cm}}
\pgfpathlineto{\pgfqpoint{0.682cm}{0.042cm}}
\pgfusepath{stroke}
\end{pgfscope}
\end{pgfscope}
\end{pgfscope}
\end{pgfscope}
\end{tikzpicture}}} + \phi + \psi$, where $\phi$ is paracontrolled by $X^{\!\resizebox{0.6em}{!}{
\begin{tikzpicture}
\pgfpathmoveto{\pgfqpoint{0cm}{0cm}}
\pgfpathlineto{\pgfqpoint{1.376cm}{0cm}}
\pgfpathlineto{\pgfqpoint{1.376cm}{1.588cm}}
\pgfpathlineto{\pgfqpoint{0cm}{1.588cm}}
\pgfpathclose
\pgfusepath{clip}
\begin{pgfscope}
\begin{pgfscope}
\pgfpathmoveto{\pgfqpoint{0cm}{0cm}}
\pgfpathlineto{\pgfqpoint{1.376cm}{0cm}}
\pgfpathlineto{\pgfqpoint{1.376cm}{1.588cm}}
\pgfpathlineto{\pgfqpoint{0cm}{1.588cm}}
\pgfpathclose
\pgfusepath{clip}
\begin{pgfscope}
\begin{pgfscope}
\definecolor{eps2pgf_color}{gray}{0.976471}\pgfsetstrokecolor{eps2pgf_color}\pgfsetfillcolor{eps2pgf_color}
\pgfpathmoveto{\pgfqpoint{0cm}{0cm}}
\pgfpathlineto{\pgfqpoint{1.376cm}{0cm}}
\pgfpathlineto{\pgfqpoint{1.376cm}{1.588cm}}
\pgfpathlineto{\pgfqpoint{0cm}{1.588cm}}
\pgfpathclose
\pgfusepath{fill}
\end{pgfscope}
\begin{pgfscope}
\pgfsetdash{}{0cm}
\pgfsetlinewidth{0.818mm}
\pgfsetroundcap
\pgfsetroundjoin
\pgfsetmiterlimit{7.0}
\definecolor{eps2pgf_color}{gray}{0}\pgfsetstrokecolor{eps2pgf_color}\pgfsetfillcolor{eps2pgf_color}
\pgfpathmoveto{\pgfqpoint{0.117cm}{1.476cm}}
\pgfpathlineto{\pgfqpoint{0.682cm}{0.726cm}}
\pgfpathlineto{\pgfqpoint{1.246cm}{1.476cm}}
\pgfusepath{stroke}
\end{pgfscope}
\definecolor{eps2pgf_color}{gray}{0}\pgfsetstrokecolor{eps2pgf_color}\pgfsetfillcolor{eps2pgf_color}
\pgfpathmoveto{\pgfqpoint{0.273cm}{1.451cm}}
\pgfpathcurveto{\pgfqpoint{0.273cm}{1.487cm}}{\pgfqpoint{0.259cm}{1.522cm}}{\pgfqpoint{0.233cm}{1.547cm}}
\pgfpathcurveto{\pgfqpoint{0.207cm}{1.573cm}}{\pgfqpoint{0.173cm}{1.588cm}}{\pgfqpoint{0.137cm}{1.588cm}}
\pgfpathcurveto{\pgfqpoint{0.1cm}{1.588cm}}{\pgfqpoint{0.066cm}{1.573cm}}{\pgfqpoint{0.04cm}{1.547cm}}
\pgfpathcurveto{\pgfqpoint{0.014cm}{1.522cm}}{\pgfqpoint{0cm}{1.487cm}}{\pgfqpoint{0cm}{1.451cm}}
\pgfpathcurveto{\pgfqpoint{0cm}{1.414cm}}{\pgfqpoint{0.014cm}{1.379cm}}{\pgfqpoint{0.04cm}{1.354cm}}
\pgfpathcurveto{\pgfqpoint{0.066cm}{1.328cm}}{\pgfqpoint{0.1cm}{1.314cm}}{\pgfqpoint{0.137cm}{1.314cm}}
\pgfpathcurveto{\pgfqpoint{0.173cm}{1.314cm}}{\pgfqpoint{0.207cm}{1.328cm}}{\pgfqpoint{0.233cm}{1.354cm}}
\pgfpathcurveto{\pgfqpoint{0.259cm}{1.379cm}}{\pgfqpoint{0.273cm}{1.414cm}}{\pgfqpoint{0.273cm}{1.451cm}}
\pgfusepath{fill}
\pgfpathmoveto{\pgfqpoint{1.345cm}{1.426cm}}
\pgfpathcurveto{\pgfqpoint{1.345cm}{1.463cm}}{\pgfqpoint{1.331cm}{1.497cm}}{\pgfqpoint{1.305cm}{1.523cm}}
\pgfpathcurveto{\pgfqpoint{1.28cm}{1.549cm}}{\pgfqpoint{1.245cm}{1.563cm}}{\pgfqpoint{1.209cm}{1.563cm}}
\pgfpathcurveto{\pgfqpoint{1.172cm}{1.563cm}}{\pgfqpoint{1.138cm}{1.549cm}}{\pgfqpoint{1.112cm}{1.523cm}}
\pgfpathcurveto{\pgfqpoint{1.087cm}{1.497cm}}{\pgfqpoint{1.072cm}{1.463cm}}{\pgfqpoint{1.072cm}{1.426cm}}
\pgfpathcurveto{\pgfqpoint{1.072cm}{1.39cm}}{\pgfqpoint{1.087cm}{1.355cm}}{\pgfqpoint{1.112cm}{1.329cm}}
\pgfpathcurveto{\pgfqpoint{1.138cm}{1.304cm}}{\pgfqpoint{1.172cm}{1.289cm}}{\pgfqpoint{1.209cm}{1.289cm}}
\pgfpathcurveto{\pgfqpoint{1.245cm}{1.289cm}}{\pgfqpoint{1.28cm}{1.304cm}}{\pgfqpoint{1.305cm}{1.329cm}}
\pgfpathcurveto{\pgfqpoint{1.331cm}{1.355cm}}{\pgfqpoint{1.345cm}{1.39cm}}{\pgfqpoint{1.345cm}{1.426cm}}
\pgfusepath{fill}
\begin{pgfscope}
\pgfsetdash{}{0cm}
\pgfsetlinewidth{0.818mm}
\pgfsetroundcap
\pgfsetmiterlimit{4.0}
\pgfpathmoveto{\pgfqpoint{0.682cm}{0.726cm}}
\pgfpathlineto{\pgfqpoint{0.682cm}{0.097cm}}
\pgfusepath{stroke}
\end{pgfscope}
\end{pgfscope}
\end{pgfscope}
\end{pgfscope}
\end{tikzpicture}}}$
 and equations \eqref{eq:two43}, \eqref{eq:th43} are satisfied.
For the purposes of the proof of uniqueness, it is not necessary to consider the modified paraproduct and therefore we may work with a similar decomposition as in the elliptic setting in Section \ref{sec:45}.
We define
  \[ \phi = \theta - 3 (- X^{\!\resizebox{0.6em}{!}{
\begin{tikzpicture}
\pgfpathmoveto{\pgfqpoint{0cm}{-0.035cm}}
\pgfpathlineto{\pgfqpoint{1.376cm}{-0.035cm}}
\pgfpathlineto{\pgfqpoint{1.376cm}{1.552cm}}
\pgfpathlineto{\pgfqpoint{0cm}{1.552cm}}
\pgfpathclose
\pgfusepath{clip}
\begin{pgfscope}
\begin{pgfscope}
\pgfpathmoveto{\pgfqpoint{0cm}{-0.035cm}}
\pgfpathlineto{\pgfqpoint{1.376cm}{-0.035cm}}
\pgfpathlineto{\pgfqpoint{1.376cm}{1.552cm}}
\pgfpathlineto{\pgfqpoint{0cm}{1.552cm}}
\pgfpathclose
\pgfusepath{clip}
\begin{pgfscope}
\begin{pgfscope}
\pgfsetdash{}{0cm}
\pgfsetlinewidth{0.818mm}
\pgfsetroundcap
\pgfsetroundjoin
\pgfsetmiterlimit{7.0}
\definecolor{eps2pgf_color}{gray}{0}\pgfsetstrokecolor{eps2pgf_color}\pgfsetfillcolor{eps2pgf_color}
\pgfpathmoveto{\pgfqpoint{0.117cm}{1.421cm}}
\pgfpathlineto{\pgfqpoint{0.682cm}{0.671cm}}
\pgfpathlineto{\pgfqpoint{1.246cm}{1.421cm}}
\pgfusepath{stroke}
\end{pgfscope}
\definecolor{eps2pgf_color}{gray}{0}\pgfsetstrokecolor{eps2pgf_color}\pgfsetfillcolor{eps2pgf_color}
\pgfpathmoveto{\pgfqpoint{0.273cm}{1.395cm}}
\pgfpathcurveto{\pgfqpoint{0.273cm}{1.432cm}}{\pgfqpoint{0.259cm}{1.467cm}}{\pgfqpoint{0.233cm}{1.492cm}}
\pgfpathcurveto{\pgfqpoint{0.207cm}{1.518cm}}{\pgfqpoint{0.173cm}{1.532cm}}{\pgfqpoint{0.137cm}{1.532cm}}
\pgfpathcurveto{\pgfqpoint{0.1cm}{1.532cm}}{\pgfqpoint{0.066cm}{1.518cm}}{\pgfqpoint{0.04cm}{1.492cm}}
\pgfpathcurveto{\pgfqpoint{0.014cm}{1.467cm}}{\pgfqpoint{0cm}{1.432cm}}{\pgfqpoint{0cm}{1.395cm}}
\pgfpathcurveto{\pgfqpoint{0cm}{1.359cm}}{\pgfqpoint{0.014cm}{1.324cm}}{\pgfqpoint{0.04cm}{1.299cm}}
\pgfpathcurveto{\pgfqpoint{0.066cm}{1.273cm}}{\pgfqpoint{0.1cm}{1.258cm}}{\pgfqpoint{0.137cm}{1.258cm}}
\pgfpathcurveto{\pgfqpoint{0.173cm}{1.258cm}}{\pgfqpoint{0.207cm}{1.273cm}}{\pgfqpoint{0.233cm}{1.299cm}}
\pgfpathcurveto{\pgfqpoint{0.259cm}{1.324cm}}{\pgfqpoint{0.273cm}{1.359cm}}{\pgfqpoint{0.273cm}{1.395cm}}
\pgfusepath{fill}
\begin{pgfscope}
\pgfsetdash{}{0cm}
\pgfsetlinewidth{0.818mm}
\pgfsetmiterlimit{7.0}
\pgfpathmoveto{\pgfqpoint{0.682cm}{0.671cm}}
\pgfpathlineto{\pgfqpoint{0.679cm}{1.418cm}}
\pgfusepath{stroke}
\end{pgfscope}
\pgfpathmoveto{\pgfqpoint{0.815cm}{1.399cm}}
\pgfpathcurveto{\pgfqpoint{0.815cm}{1.435cm}}{\pgfqpoint{0.801cm}{1.47cm}}{\pgfqpoint{0.775cm}{1.496cm}}
\pgfpathcurveto{\pgfqpoint{0.75cm}{1.521cm}}{\pgfqpoint{0.715cm}{1.536cm}}{\pgfqpoint{0.679cm}{1.536cm}}
\pgfpathcurveto{\pgfqpoint{0.643cm}{1.536cm}}{\pgfqpoint{0.608cm}{1.521cm}}{\pgfqpoint{0.582cm}{1.496cm}}
\pgfpathcurveto{\pgfqpoint{0.557cm}{1.47cm}}{\pgfqpoint{0.542cm}{1.435cm}}{\pgfqpoint{0.542cm}{1.399cm}}
\pgfpathcurveto{\pgfqpoint{0.542cm}{1.363cm}}{\pgfqpoint{0.557cm}{1.328cm}}{\pgfqpoint{0.582cm}{1.302cm}}
\pgfpathcurveto{\pgfqpoint{0.608cm}{1.276cm}}{\pgfqpoint{0.643cm}{1.262cm}}{\pgfqpoint{0.679cm}{1.262cm}}
\pgfpathcurveto{\pgfqpoint{0.715cm}{1.262cm}}{\pgfqpoint{0.75cm}{1.276cm}}{\pgfqpoint{0.775cm}{1.302cm}}
\pgfpathcurveto{\pgfqpoint{0.801cm}{1.328cm}}{\pgfqpoint{0.815cm}{1.363cm}}{\pgfqpoint{0.815cm}{1.399cm}}
\pgfusepath{fill}
\pgfpathmoveto{\pgfqpoint{1.345cm}{1.371cm}}
\pgfpathcurveto{\pgfqpoint{1.345cm}{1.408cm}}{\pgfqpoint{1.331cm}{1.442cm}}{\pgfqpoint{1.305cm}{1.468cm}}
\pgfpathcurveto{\pgfqpoint{1.28cm}{1.494cm}}{\pgfqpoint{1.245cm}{1.508cm}}{\pgfqpoint{1.209cm}{1.508cm}}
\pgfpathcurveto{\pgfqpoint{1.172cm}{1.508cm}}{\pgfqpoint{1.138cm}{1.494cm}}{\pgfqpoint{1.112cm}{1.468cm}}
\pgfpathcurveto{\pgfqpoint{1.087cm}{1.442cm}}{\pgfqpoint{1.072cm}{1.408cm}}{\pgfqpoint{1.072cm}{1.371cm}}
\pgfpathcurveto{\pgfqpoint{1.072cm}{1.335cm}}{\pgfqpoint{1.087cm}{1.3cm}}{\pgfqpoint{1.112cm}{1.274cm}}
\pgfpathcurveto{\pgfqpoint{1.138cm}{1.249cm}}{\pgfqpoint{1.172cm}{1.234cm}}{\pgfqpoint{1.209cm}{1.234cm}}
\pgfpathcurveto{\pgfqpoint{1.245cm}{1.234cm}}{\pgfqpoint{1.28cm}{1.249cm}}{\pgfqpoint{1.305cm}{1.274cm}}
\pgfpathcurveto{\pgfqpoint{1.331cm}{1.3cm}}{\pgfqpoint{1.345cm}{1.335cm}}{\pgfqpoint{1.345cm}{1.371cm}}
\pgfusepath{fill}
\begin{pgfscope}
\pgfsetdash{}{0cm}
\pgfsetlinewidth{0.818mm}
\pgfsetroundcap
\pgfsetmiterlimit{4.0}
\pgfpathmoveto{\pgfqpoint{0.682cm}{0.671cm}}
\pgfpathlineto{\pgfqpoint{0.682cm}{0.042cm}}
\pgfusepath{stroke}
\end{pgfscope}
\end{pgfscope}
\end{pgfscope}
\end{pgfscope}
\end{tikzpicture}}} + \phi + \psi) \prec
     X^{\!\resizebox{0.6em}{!}{
\begin{tikzpicture}
\pgfpathmoveto{\pgfqpoint{0cm}{0cm}}
\pgfpathlineto{\pgfqpoint{1.376cm}{0cm}}
\pgfpathlineto{\pgfqpoint{1.376cm}{1.588cm}}
\pgfpathlineto{\pgfqpoint{0cm}{1.588cm}}
\pgfpathclose
\pgfusepath{clip}
\begin{pgfscope}
\begin{pgfscope}
\pgfpathmoveto{\pgfqpoint{0cm}{0cm}}
\pgfpathlineto{\pgfqpoint{1.376cm}{0cm}}
\pgfpathlineto{\pgfqpoint{1.376cm}{1.588cm}}
\pgfpathlineto{\pgfqpoint{0cm}{1.588cm}}
\pgfpathclose
\pgfusepath{clip}
\begin{pgfscope}
\begin{pgfscope}
\definecolor{eps2pgf_color}{gray}{0.976471}\pgfsetstrokecolor{eps2pgf_color}\pgfsetfillcolor{eps2pgf_color}
\pgfpathmoveto{\pgfqpoint{0cm}{0cm}}
\pgfpathlineto{\pgfqpoint{1.376cm}{0cm}}
\pgfpathlineto{\pgfqpoint{1.376cm}{1.588cm}}
\pgfpathlineto{\pgfqpoint{0cm}{1.588cm}}
\pgfpathclose
\pgfusepath{fill}
\end{pgfscope}
\begin{pgfscope}
\pgfsetdash{}{0cm}
\pgfsetlinewidth{0.818mm}
\pgfsetroundcap
\pgfsetroundjoin
\pgfsetmiterlimit{7.0}
\definecolor{eps2pgf_color}{gray}{0}\pgfsetstrokecolor{eps2pgf_color}\pgfsetfillcolor{eps2pgf_color}
\pgfpathmoveto{\pgfqpoint{0.117cm}{1.476cm}}
\pgfpathlineto{\pgfqpoint{0.682cm}{0.726cm}}
\pgfpathlineto{\pgfqpoint{1.246cm}{1.476cm}}
\pgfusepath{stroke}
\end{pgfscope}
\definecolor{eps2pgf_color}{gray}{0}\pgfsetstrokecolor{eps2pgf_color}\pgfsetfillcolor{eps2pgf_color}
\pgfpathmoveto{\pgfqpoint{0.273cm}{1.451cm}}
\pgfpathcurveto{\pgfqpoint{0.273cm}{1.487cm}}{\pgfqpoint{0.259cm}{1.522cm}}{\pgfqpoint{0.233cm}{1.547cm}}
\pgfpathcurveto{\pgfqpoint{0.207cm}{1.573cm}}{\pgfqpoint{0.173cm}{1.588cm}}{\pgfqpoint{0.137cm}{1.588cm}}
\pgfpathcurveto{\pgfqpoint{0.1cm}{1.588cm}}{\pgfqpoint{0.066cm}{1.573cm}}{\pgfqpoint{0.04cm}{1.547cm}}
\pgfpathcurveto{\pgfqpoint{0.014cm}{1.522cm}}{\pgfqpoint{0cm}{1.487cm}}{\pgfqpoint{0cm}{1.451cm}}
\pgfpathcurveto{\pgfqpoint{0cm}{1.414cm}}{\pgfqpoint{0.014cm}{1.379cm}}{\pgfqpoint{0.04cm}{1.354cm}}
\pgfpathcurveto{\pgfqpoint{0.066cm}{1.328cm}}{\pgfqpoint{0.1cm}{1.314cm}}{\pgfqpoint{0.137cm}{1.314cm}}
\pgfpathcurveto{\pgfqpoint{0.173cm}{1.314cm}}{\pgfqpoint{0.207cm}{1.328cm}}{\pgfqpoint{0.233cm}{1.354cm}}
\pgfpathcurveto{\pgfqpoint{0.259cm}{1.379cm}}{\pgfqpoint{0.273cm}{1.414cm}}{\pgfqpoint{0.273cm}{1.451cm}}
\pgfusepath{fill}
\pgfpathmoveto{\pgfqpoint{1.345cm}{1.426cm}}
\pgfpathcurveto{\pgfqpoint{1.345cm}{1.463cm}}{\pgfqpoint{1.331cm}{1.497cm}}{\pgfqpoint{1.305cm}{1.523cm}}
\pgfpathcurveto{\pgfqpoint{1.28cm}{1.549cm}}{\pgfqpoint{1.245cm}{1.563cm}}{\pgfqpoint{1.209cm}{1.563cm}}
\pgfpathcurveto{\pgfqpoint{1.172cm}{1.563cm}}{\pgfqpoint{1.138cm}{1.549cm}}{\pgfqpoint{1.112cm}{1.523cm}}
\pgfpathcurveto{\pgfqpoint{1.087cm}{1.497cm}}{\pgfqpoint{1.072cm}{1.463cm}}{\pgfqpoint{1.072cm}{1.426cm}}
\pgfpathcurveto{\pgfqpoint{1.072cm}{1.39cm}}{\pgfqpoint{1.087cm}{1.355cm}}{\pgfqpoint{1.112cm}{1.329cm}}
\pgfpathcurveto{\pgfqpoint{1.138cm}{1.304cm}}{\pgfqpoint{1.172cm}{1.289cm}}{\pgfqpoint{1.209cm}{1.289cm}}
\pgfpathcurveto{\pgfqpoint{1.245cm}{1.289cm}}{\pgfqpoint{1.28cm}{1.304cm}}{\pgfqpoint{1.305cm}{1.329cm}}
\pgfpathcurveto{\pgfqpoint{1.331cm}{1.355cm}}{\pgfqpoint{1.345cm}{1.39cm}}{\pgfqpoint{1.345cm}{1.426cm}}
\pgfusepath{fill}
\begin{pgfscope}
\pgfsetdash{}{0cm}
\pgfsetlinewidth{0.818mm}
\pgfsetroundcap
\pgfsetmiterlimit{4.0}
\pgfpathmoveto{\pgfqpoint{0.682cm}{0.726cm}}
\pgfpathlineto{\pgfqpoint{0.682cm}{0.097cm}}
\pgfusepath{stroke}
\end{pgfscope}
\end{pgfscope}
\end{pgfscope}
\end{pgfscope}
\end{tikzpicture}}} \]
  where (formally)
  \begin{align*}
      0 & = \LL \theta + \LL \psi  + 3 \llbracket X^2 \rrbracket \preccurlyeq (- X^{\!\resizebox{0.6em}{!}{
\begin{tikzpicture}
\pgfpathmoveto{\pgfqpoint{0cm}{-0.035cm}}
\pgfpathlineto{\pgfqpoint{1.376cm}{-0.035cm}}
\pgfpathlineto{\pgfqpoint{1.376cm}{1.552cm}}
\pgfpathlineto{\pgfqpoint{0cm}{1.552cm}}
\pgfpathclose
\pgfusepath{clip}
\begin{pgfscope}
\begin{pgfscope}
\pgfpathmoveto{\pgfqpoint{0cm}{-0.035cm}}
\pgfpathlineto{\pgfqpoint{1.376cm}{-0.035cm}}
\pgfpathlineto{\pgfqpoint{1.376cm}{1.552cm}}
\pgfpathlineto{\pgfqpoint{0cm}{1.552cm}}
\pgfpathclose
\pgfusepath{clip}
\begin{pgfscope}
\begin{pgfscope}
\pgfsetdash{}{0cm}
\pgfsetlinewidth{0.818mm}
\pgfsetroundcap
\pgfsetroundjoin
\pgfsetmiterlimit{7.0}
\definecolor{eps2pgf_color}{gray}{0}\pgfsetstrokecolor{eps2pgf_color}\pgfsetfillcolor{eps2pgf_color}
\pgfpathmoveto{\pgfqpoint{0.117cm}{1.421cm}}
\pgfpathlineto{\pgfqpoint{0.682cm}{0.671cm}}
\pgfpathlineto{\pgfqpoint{1.246cm}{1.421cm}}
\pgfusepath{stroke}
\end{pgfscope}
\definecolor{eps2pgf_color}{gray}{0}\pgfsetstrokecolor{eps2pgf_color}\pgfsetfillcolor{eps2pgf_color}
\pgfpathmoveto{\pgfqpoint{0.273cm}{1.395cm}}
\pgfpathcurveto{\pgfqpoint{0.273cm}{1.432cm}}{\pgfqpoint{0.259cm}{1.467cm}}{\pgfqpoint{0.233cm}{1.492cm}}
\pgfpathcurveto{\pgfqpoint{0.207cm}{1.518cm}}{\pgfqpoint{0.173cm}{1.532cm}}{\pgfqpoint{0.137cm}{1.532cm}}
\pgfpathcurveto{\pgfqpoint{0.1cm}{1.532cm}}{\pgfqpoint{0.066cm}{1.518cm}}{\pgfqpoint{0.04cm}{1.492cm}}
\pgfpathcurveto{\pgfqpoint{0.014cm}{1.467cm}}{\pgfqpoint{0cm}{1.432cm}}{\pgfqpoint{0cm}{1.395cm}}
\pgfpathcurveto{\pgfqpoint{0cm}{1.359cm}}{\pgfqpoint{0.014cm}{1.324cm}}{\pgfqpoint{0.04cm}{1.299cm}}
\pgfpathcurveto{\pgfqpoint{0.066cm}{1.273cm}}{\pgfqpoint{0.1cm}{1.258cm}}{\pgfqpoint{0.137cm}{1.258cm}}
\pgfpathcurveto{\pgfqpoint{0.173cm}{1.258cm}}{\pgfqpoint{0.207cm}{1.273cm}}{\pgfqpoint{0.233cm}{1.299cm}}
\pgfpathcurveto{\pgfqpoint{0.259cm}{1.324cm}}{\pgfqpoint{0.273cm}{1.359cm}}{\pgfqpoint{0.273cm}{1.395cm}}
\pgfusepath{fill}
\begin{pgfscope}
\pgfsetdash{}{0cm}
\pgfsetlinewidth{0.818mm}
\pgfsetmiterlimit{7.0}
\pgfpathmoveto{\pgfqpoint{0.682cm}{0.671cm}}
\pgfpathlineto{\pgfqpoint{0.679cm}{1.418cm}}
\pgfusepath{stroke}
\end{pgfscope}
\pgfpathmoveto{\pgfqpoint{0.815cm}{1.399cm}}
\pgfpathcurveto{\pgfqpoint{0.815cm}{1.435cm}}{\pgfqpoint{0.801cm}{1.47cm}}{\pgfqpoint{0.775cm}{1.496cm}}
\pgfpathcurveto{\pgfqpoint{0.75cm}{1.521cm}}{\pgfqpoint{0.715cm}{1.536cm}}{\pgfqpoint{0.679cm}{1.536cm}}
\pgfpathcurveto{\pgfqpoint{0.643cm}{1.536cm}}{\pgfqpoint{0.608cm}{1.521cm}}{\pgfqpoint{0.582cm}{1.496cm}}
\pgfpathcurveto{\pgfqpoint{0.557cm}{1.47cm}}{\pgfqpoint{0.542cm}{1.435cm}}{\pgfqpoint{0.542cm}{1.399cm}}
\pgfpathcurveto{\pgfqpoint{0.542cm}{1.363cm}}{\pgfqpoint{0.557cm}{1.328cm}}{\pgfqpoint{0.582cm}{1.302cm}}
\pgfpathcurveto{\pgfqpoint{0.608cm}{1.276cm}}{\pgfqpoint{0.643cm}{1.262cm}}{\pgfqpoint{0.679cm}{1.262cm}}
\pgfpathcurveto{\pgfqpoint{0.715cm}{1.262cm}}{\pgfqpoint{0.75cm}{1.276cm}}{\pgfqpoint{0.775cm}{1.302cm}}
\pgfpathcurveto{\pgfqpoint{0.801cm}{1.328cm}}{\pgfqpoint{0.815cm}{1.363cm}}{\pgfqpoint{0.815cm}{1.399cm}}
\pgfusepath{fill}
\pgfpathmoveto{\pgfqpoint{1.345cm}{1.371cm}}
\pgfpathcurveto{\pgfqpoint{1.345cm}{1.408cm}}{\pgfqpoint{1.331cm}{1.442cm}}{\pgfqpoint{1.305cm}{1.468cm}}
\pgfpathcurveto{\pgfqpoint{1.28cm}{1.494cm}}{\pgfqpoint{1.245cm}{1.508cm}}{\pgfqpoint{1.209cm}{1.508cm}}
\pgfpathcurveto{\pgfqpoint{1.172cm}{1.508cm}}{\pgfqpoint{1.138cm}{1.494cm}}{\pgfqpoint{1.112cm}{1.468cm}}
\pgfpathcurveto{\pgfqpoint{1.087cm}{1.442cm}}{\pgfqpoint{1.072cm}{1.408cm}}{\pgfqpoint{1.072cm}{1.371cm}}
\pgfpathcurveto{\pgfqpoint{1.072cm}{1.335cm}}{\pgfqpoint{1.087cm}{1.3cm}}{\pgfqpoint{1.112cm}{1.274cm}}
\pgfpathcurveto{\pgfqpoint{1.138cm}{1.249cm}}{\pgfqpoint{1.172cm}{1.234cm}}{\pgfqpoint{1.209cm}{1.234cm}}
\pgfpathcurveto{\pgfqpoint{1.245cm}{1.234cm}}{\pgfqpoint{1.28cm}{1.249cm}}{\pgfqpoint{1.305cm}{1.274cm}}
\pgfpathcurveto{\pgfqpoint{1.331cm}{1.3cm}}{\pgfqpoint{1.345cm}{1.335cm}}{\pgfqpoint{1.345cm}{1.371cm}}
\pgfusepath{fill}
\begin{pgfscope}
\pgfsetdash{}{0cm}
\pgfsetlinewidth{0.818mm}
\pgfsetroundcap
\pgfsetmiterlimit{4.0}
\pgfpathmoveto{\pgfqpoint{0.682cm}{0.671cm}}
\pgfpathlineto{\pgfqpoint{0.682cm}{0.042cm}}
\pgfusepath{stroke}
\end{pgfscope}
\end{pgfscope}
\end{pgfscope}
\end{pgfscope}
\end{tikzpicture}}} +
      \phi + \psi) - 3 [\LL, (- X^{\!\resizebox{0.6em}{!}{
\begin{tikzpicture}
\pgfpathmoveto{\pgfqpoint{0cm}{-0.035cm}}
\pgfpathlineto{\pgfqpoint{1.376cm}{-0.035cm}}
\pgfpathlineto{\pgfqpoint{1.376cm}{1.552cm}}
\pgfpathlineto{\pgfqpoint{0cm}{1.552cm}}
\pgfpathclose
\pgfusepath{clip}
\begin{pgfscope}
\begin{pgfscope}
\pgfpathmoveto{\pgfqpoint{0cm}{-0.035cm}}
\pgfpathlineto{\pgfqpoint{1.376cm}{-0.035cm}}
\pgfpathlineto{\pgfqpoint{1.376cm}{1.552cm}}
\pgfpathlineto{\pgfqpoint{0cm}{1.552cm}}
\pgfpathclose
\pgfusepath{clip}
\begin{pgfscope}
\begin{pgfscope}
\pgfsetdash{}{0cm}
\pgfsetlinewidth{0.818mm}
\pgfsetroundcap
\pgfsetroundjoin
\pgfsetmiterlimit{7.0}
\definecolor{eps2pgf_color}{gray}{0}\pgfsetstrokecolor{eps2pgf_color}\pgfsetfillcolor{eps2pgf_color}
\pgfpathmoveto{\pgfqpoint{0.117cm}{1.421cm}}
\pgfpathlineto{\pgfqpoint{0.682cm}{0.671cm}}
\pgfpathlineto{\pgfqpoint{1.246cm}{1.421cm}}
\pgfusepath{stroke}
\end{pgfscope}
\definecolor{eps2pgf_color}{gray}{0}\pgfsetstrokecolor{eps2pgf_color}\pgfsetfillcolor{eps2pgf_color}
\pgfpathmoveto{\pgfqpoint{0.273cm}{1.395cm}}
\pgfpathcurveto{\pgfqpoint{0.273cm}{1.432cm}}{\pgfqpoint{0.259cm}{1.467cm}}{\pgfqpoint{0.233cm}{1.492cm}}
\pgfpathcurveto{\pgfqpoint{0.207cm}{1.518cm}}{\pgfqpoint{0.173cm}{1.532cm}}{\pgfqpoint{0.137cm}{1.532cm}}
\pgfpathcurveto{\pgfqpoint{0.1cm}{1.532cm}}{\pgfqpoint{0.066cm}{1.518cm}}{\pgfqpoint{0.04cm}{1.492cm}}
\pgfpathcurveto{\pgfqpoint{0.014cm}{1.467cm}}{\pgfqpoint{0cm}{1.432cm}}{\pgfqpoint{0cm}{1.395cm}}
\pgfpathcurveto{\pgfqpoint{0cm}{1.359cm}}{\pgfqpoint{0.014cm}{1.324cm}}{\pgfqpoint{0.04cm}{1.299cm}}
\pgfpathcurveto{\pgfqpoint{0.066cm}{1.273cm}}{\pgfqpoint{0.1cm}{1.258cm}}{\pgfqpoint{0.137cm}{1.258cm}}
\pgfpathcurveto{\pgfqpoint{0.173cm}{1.258cm}}{\pgfqpoint{0.207cm}{1.273cm}}{\pgfqpoint{0.233cm}{1.299cm}}
\pgfpathcurveto{\pgfqpoint{0.259cm}{1.324cm}}{\pgfqpoint{0.273cm}{1.359cm}}{\pgfqpoint{0.273cm}{1.395cm}}
\pgfusepath{fill}
\begin{pgfscope}
\pgfsetdash{}{0cm}
\pgfsetlinewidth{0.818mm}
\pgfsetmiterlimit{7.0}
\pgfpathmoveto{\pgfqpoint{0.682cm}{0.671cm}}
\pgfpathlineto{\pgfqpoint{0.679cm}{1.418cm}}
\pgfusepath{stroke}
\end{pgfscope}
\pgfpathmoveto{\pgfqpoint{0.815cm}{1.399cm}}
\pgfpathcurveto{\pgfqpoint{0.815cm}{1.435cm}}{\pgfqpoint{0.801cm}{1.47cm}}{\pgfqpoint{0.775cm}{1.496cm}}
\pgfpathcurveto{\pgfqpoint{0.75cm}{1.521cm}}{\pgfqpoint{0.715cm}{1.536cm}}{\pgfqpoint{0.679cm}{1.536cm}}
\pgfpathcurveto{\pgfqpoint{0.643cm}{1.536cm}}{\pgfqpoint{0.608cm}{1.521cm}}{\pgfqpoint{0.582cm}{1.496cm}}
\pgfpathcurveto{\pgfqpoint{0.557cm}{1.47cm}}{\pgfqpoint{0.542cm}{1.435cm}}{\pgfqpoint{0.542cm}{1.399cm}}
\pgfpathcurveto{\pgfqpoint{0.542cm}{1.363cm}}{\pgfqpoint{0.557cm}{1.328cm}}{\pgfqpoint{0.582cm}{1.302cm}}
\pgfpathcurveto{\pgfqpoint{0.608cm}{1.276cm}}{\pgfqpoint{0.643cm}{1.262cm}}{\pgfqpoint{0.679cm}{1.262cm}}
\pgfpathcurveto{\pgfqpoint{0.715cm}{1.262cm}}{\pgfqpoint{0.75cm}{1.276cm}}{\pgfqpoint{0.775cm}{1.302cm}}
\pgfpathcurveto{\pgfqpoint{0.801cm}{1.328cm}}{\pgfqpoint{0.815cm}{1.363cm}}{\pgfqpoint{0.815cm}{1.399cm}}
\pgfusepath{fill}
\pgfpathmoveto{\pgfqpoint{1.345cm}{1.371cm}}
\pgfpathcurveto{\pgfqpoint{1.345cm}{1.408cm}}{\pgfqpoint{1.331cm}{1.442cm}}{\pgfqpoint{1.305cm}{1.468cm}}
\pgfpathcurveto{\pgfqpoint{1.28cm}{1.494cm}}{\pgfqpoint{1.245cm}{1.508cm}}{\pgfqpoint{1.209cm}{1.508cm}}
\pgfpathcurveto{\pgfqpoint{1.172cm}{1.508cm}}{\pgfqpoint{1.138cm}{1.494cm}}{\pgfqpoint{1.112cm}{1.468cm}}
\pgfpathcurveto{\pgfqpoint{1.087cm}{1.442cm}}{\pgfqpoint{1.072cm}{1.408cm}}{\pgfqpoint{1.072cm}{1.371cm}}
\pgfpathcurveto{\pgfqpoint{1.072cm}{1.335cm}}{\pgfqpoint{1.087cm}{1.3cm}}{\pgfqpoint{1.112cm}{1.274cm}}
\pgfpathcurveto{\pgfqpoint{1.138cm}{1.249cm}}{\pgfqpoint{1.172cm}{1.234cm}}{\pgfqpoint{1.209cm}{1.234cm}}
\pgfpathcurveto{\pgfqpoint{1.245cm}{1.234cm}}{\pgfqpoint{1.28cm}{1.249cm}}{\pgfqpoint{1.305cm}{1.274cm}}
\pgfpathcurveto{\pgfqpoint{1.331cm}{1.3cm}}{\pgfqpoint{1.345cm}{1.335cm}}{\pgfqpoint{1.345cm}{1.371cm}}
\pgfusepath{fill}
\begin{pgfscope}
\pgfsetdash{}{0cm}
\pgfsetlinewidth{0.818mm}
\pgfsetroundcap
\pgfsetmiterlimit{4.0}
\pgfpathmoveto{\pgfqpoint{0.682cm}{0.671cm}}
\pgfpathlineto{\pgfqpoint{0.682cm}{0.042cm}}
\pgfusepath{stroke}
\end{pgfscope}
\end{pgfscope}
\end{pgfscope}
\end{pgfscope}
\end{tikzpicture}}} + \phi + \psi) \prec]
      X^{\!\resizebox{0.6em}{!}{
\begin{tikzpicture}
\pgfpathmoveto{\pgfqpoint{0cm}{0cm}}
\pgfpathlineto{\pgfqpoint{1.376cm}{0cm}}
\pgfpathlineto{\pgfqpoint{1.376cm}{1.588cm}}
\pgfpathlineto{\pgfqpoint{0cm}{1.588cm}}
\pgfpathclose
\pgfusepath{clip}
\begin{pgfscope}
\begin{pgfscope}
\pgfpathmoveto{\pgfqpoint{0cm}{0cm}}
\pgfpathlineto{\pgfqpoint{1.376cm}{0cm}}
\pgfpathlineto{\pgfqpoint{1.376cm}{1.588cm}}
\pgfpathlineto{\pgfqpoint{0cm}{1.588cm}}
\pgfpathclose
\pgfusepath{clip}
\begin{pgfscope}
\begin{pgfscope}
\definecolor{eps2pgf_color}{gray}{0.976471}\pgfsetstrokecolor{eps2pgf_color}\pgfsetfillcolor{eps2pgf_color}
\pgfpathmoveto{\pgfqpoint{0cm}{0cm}}
\pgfpathlineto{\pgfqpoint{1.376cm}{0cm}}
\pgfpathlineto{\pgfqpoint{1.376cm}{1.588cm}}
\pgfpathlineto{\pgfqpoint{0cm}{1.588cm}}
\pgfpathclose
\pgfusepath{fill}
\end{pgfscope}
\begin{pgfscope}
\pgfsetdash{}{0cm}
\pgfsetlinewidth{0.818mm}
\pgfsetroundcap
\pgfsetroundjoin
\pgfsetmiterlimit{7.0}
\definecolor{eps2pgf_color}{gray}{0}\pgfsetstrokecolor{eps2pgf_color}\pgfsetfillcolor{eps2pgf_color}
\pgfpathmoveto{\pgfqpoint{0.117cm}{1.476cm}}
\pgfpathlineto{\pgfqpoint{0.682cm}{0.726cm}}
\pgfpathlineto{\pgfqpoint{1.246cm}{1.476cm}}
\pgfusepath{stroke}
\end{pgfscope}
\definecolor{eps2pgf_color}{gray}{0}\pgfsetstrokecolor{eps2pgf_color}\pgfsetfillcolor{eps2pgf_color}
\pgfpathmoveto{\pgfqpoint{0.273cm}{1.451cm}}
\pgfpathcurveto{\pgfqpoint{0.273cm}{1.487cm}}{\pgfqpoint{0.259cm}{1.522cm}}{\pgfqpoint{0.233cm}{1.547cm}}
\pgfpathcurveto{\pgfqpoint{0.207cm}{1.573cm}}{\pgfqpoint{0.173cm}{1.588cm}}{\pgfqpoint{0.137cm}{1.588cm}}
\pgfpathcurveto{\pgfqpoint{0.1cm}{1.588cm}}{\pgfqpoint{0.066cm}{1.573cm}}{\pgfqpoint{0.04cm}{1.547cm}}
\pgfpathcurveto{\pgfqpoint{0.014cm}{1.522cm}}{\pgfqpoint{0cm}{1.487cm}}{\pgfqpoint{0cm}{1.451cm}}
\pgfpathcurveto{\pgfqpoint{0cm}{1.414cm}}{\pgfqpoint{0.014cm}{1.379cm}}{\pgfqpoint{0.04cm}{1.354cm}}
\pgfpathcurveto{\pgfqpoint{0.066cm}{1.328cm}}{\pgfqpoint{0.1cm}{1.314cm}}{\pgfqpoint{0.137cm}{1.314cm}}
\pgfpathcurveto{\pgfqpoint{0.173cm}{1.314cm}}{\pgfqpoint{0.207cm}{1.328cm}}{\pgfqpoint{0.233cm}{1.354cm}}
\pgfpathcurveto{\pgfqpoint{0.259cm}{1.379cm}}{\pgfqpoint{0.273cm}{1.414cm}}{\pgfqpoint{0.273cm}{1.451cm}}
\pgfusepath{fill}
\pgfpathmoveto{\pgfqpoint{1.345cm}{1.426cm}}
\pgfpathcurveto{\pgfqpoint{1.345cm}{1.463cm}}{\pgfqpoint{1.331cm}{1.497cm}}{\pgfqpoint{1.305cm}{1.523cm}}
\pgfpathcurveto{\pgfqpoint{1.28cm}{1.549cm}}{\pgfqpoint{1.245cm}{1.563cm}}{\pgfqpoint{1.209cm}{1.563cm}}
\pgfpathcurveto{\pgfqpoint{1.172cm}{1.563cm}}{\pgfqpoint{1.138cm}{1.549cm}}{\pgfqpoint{1.112cm}{1.523cm}}
\pgfpathcurveto{\pgfqpoint{1.087cm}{1.497cm}}{\pgfqpoint{1.072cm}{1.463cm}}{\pgfqpoint{1.072cm}{1.426cm}}
\pgfpathcurveto{\pgfqpoint{1.072cm}{1.39cm}}{\pgfqpoint{1.087cm}{1.355cm}}{\pgfqpoint{1.112cm}{1.329cm}}
\pgfpathcurveto{\pgfqpoint{1.138cm}{1.304cm}}{\pgfqpoint{1.172cm}{1.289cm}}{\pgfqpoint{1.209cm}{1.289cm}}
\pgfpathcurveto{\pgfqpoint{1.245cm}{1.289cm}}{\pgfqpoint{1.28cm}{1.304cm}}{\pgfqpoint{1.305cm}{1.329cm}}
\pgfpathcurveto{\pgfqpoint{1.331cm}{1.355cm}}{\pgfqpoint{1.345cm}{1.39cm}}{\pgfqpoint{1.345cm}{1.426cm}}
\pgfusepath{fill}
\begin{pgfscope}
\pgfsetdash{}{0cm}
\pgfsetlinewidth{0.818mm}
\pgfsetroundcap
\pgfsetmiterlimit{4.0}
\pgfpathmoveto{\pgfqpoint{0.682cm}{0.726cm}}
\pgfpathlineto{\pgfqpoint{0.682cm}{0.097cm}}
\pgfusepath{stroke}
\end{pgfscope}
\end{pgfscope}
\end{pgfscope}
\end{pgfscope}
\end{tikzpicture}}}\\
      & \quad + 3 X (- X^{\!\resizebox{0.6em}{!}{
\begin{tikzpicture}
\pgfpathmoveto{\pgfqpoint{0cm}{-0.035cm}}
\pgfpathlineto{\pgfqpoint{1.376cm}{-0.035cm}}
\pgfpathlineto{\pgfqpoint{1.376cm}{1.552cm}}
\pgfpathlineto{\pgfqpoint{0cm}{1.552cm}}
\pgfpathclose
\pgfusepath{clip}
\begin{pgfscope}
\begin{pgfscope}
\pgfpathmoveto{\pgfqpoint{0cm}{-0.035cm}}
\pgfpathlineto{\pgfqpoint{1.376cm}{-0.035cm}}
\pgfpathlineto{\pgfqpoint{1.376cm}{1.552cm}}
\pgfpathlineto{\pgfqpoint{0cm}{1.552cm}}
\pgfpathclose
\pgfusepath{clip}
\begin{pgfscope}
\begin{pgfscope}
\pgfsetdash{}{0cm}
\pgfsetlinewidth{0.818mm}
\pgfsetroundcap
\pgfsetroundjoin
\pgfsetmiterlimit{7.0}
\definecolor{eps2pgf_color}{gray}{0}\pgfsetstrokecolor{eps2pgf_color}\pgfsetfillcolor{eps2pgf_color}
\pgfpathmoveto{\pgfqpoint{0.117cm}{1.421cm}}
\pgfpathlineto{\pgfqpoint{0.682cm}{0.671cm}}
\pgfpathlineto{\pgfqpoint{1.246cm}{1.421cm}}
\pgfusepath{stroke}
\end{pgfscope}
\definecolor{eps2pgf_color}{gray}{0}\pgfsetstrokecolor{eps2pgf_color}\pgfsetfillcolor{eps2pgf_color}
\pgfpathmoveto{\pgfqpoint{0.273cm}{1.395cm}}
\pgfpathcurveto{\pgfqpoint{0.273cm}{1.432cm}}{\pgfqpoint{0.259cm}{1.467cm}}{\pgfqpoint{0.233cm}{1.492cm}}
\pgfpathcurveto{\pgfqpoint{0.207cm}{1.518cm}}{\pgfqpoint{0.173cm}{1.532cm}}{\pgfqpoint{0.137cm}{1.532cm}}
\pgfpathcurveto{\pgfqpoint{0.1cm}{1.532cm}}{\pgfqpoint{0.066cm}{1.518cm}}{\pgfqpoint{0.04cm}{1.492cm}}
\pgfpathcurveto{\pgfqpoint{0.014cm}{1.467cm}}{\pgfqpoint{0cm}{1.432cm}}{\pgfqpoint{0cm}{1.395cm}}
\pgfpathcurveto{\pgfqpoint{0cm}{1.359cm}}{\pgfqpoint{0.014cm}{1.324cm}}{\pgfqpoint{0.04cm}{1.299cm}}
\pgfpathcurveto{\pgfqpoint{0.066cm}{1.273cm}}{\pgfqpoint{0.1cm}{1.258cm}}{\pgfqpoint{0.137cm}{1.258cm}}
\pgfpathcurveto{\pgfqpoint{0.173cm}{1.258cm}}{\pgfqpoint{0.207cm}{1.273cm}}{\pgfqpoint{0.233cm}{1.299cm}}
\pgfpathcurveto{\pgfqpoint{0.259cm}{1.324cm}}{\pgfqpoint{0.273cm}{1.359cm}}{\pgfqpoint{0.273cm}{1.395cm}}
\pgfusepath{fill}
\begin{pgfscope}
\pgfsetdash{}{0cm}
\pgfsetlinewidth{0.818mm}
\pgfsetmiterlimit{7.0}
\pgfpathmoveto{\pgfqpoint{0.682cm}{0.671cm}}
\pgfpathlineto{\pgfqpoint{0.679cm}{1.418cm}}
\pgfusepath{stroke}
\end{pgfscope}
\pgfpathmoveto{\pgfqpoint{0.815cm}{1.399cm}}
\pgfpathcurveto{\pgfqpoint{0.815cm}{1.435cm}}{\pgfqpoint{0.801cm}{1.47cm}}{\pgfqpoint{0.775cm}{1.496cm}}
\pgfpathcurveto{\pgfqpoint{0.75cm}{1.521cm}}{\pgfqpoint{0.715cm}{1.536cm}}{\pgfqpoint{0.679cm}{1.536cm}}
\pgfpathcurveto{\pgfqpoint{0.643cm}{1.536cm}}{\pgfqpoint{0.608cm}{1.521cm}}{\pgfqpoint{0.582cm}{1.496cm}}
\pgfpathcurveto{\pgfqpoint{0.557cm}{1.47cm}}{\pgfqpoint{0.542cm}{1.435cm}}{\pgfqpoint{0.542cm}{1.399cm}}
\pgfpathcurveto{\pgfqpoint{0.542cm}{1.363cm}}{\pgfqpoint{0.557cm}{1.328cm}}{\pgfqpoint{0.582cm}{1.302cm}}
\pgfpathcurveto{\pgfqpoint{0.608cm}{1.276cm}}{\pgfqpoint{0.643cm}{1.262cm}}{\pgfqpoint{0.679cm}{1.262cm}}
\pgfpathcurveto{\pgfqpoint{0.715cm}{1.262cm}}{\pgfqpoint{0.75cm}{1.276cm}}{\pgfqpoint{0.775cm}{1.302cm}}
\pgfpathcurveto{\pgfqpoint{0.801cm}{1.328cm}}{\pgfqpoint{0.815cm}{1.363cm}}{\pgfqpoint{0.815cm}{1.399cm}}
\pgfusepath{fill}
\pgfpathmoveto{\pgfqpoint{1.345cm}{1.371cm}}
\pgfpathcurveto{\pgfqpoint{1.345cm}{1.408cm}}{\pgfqpoint{1.331cm}{1.442cm}}{\pgfqpoint{1.305cm}{1.468cm}}
\pgfpathcurveto{\pgfqpoint{1.28cm}{1.494cm}}{\pgfqpoint{1.245cm}{1.508cm}}{\pgfqpoint{1.209cm}{1.508cm}}
\pgfpathcurveto{\pgfqpoint{1.172cm}{1.508cm}}{\pgfqpoint{1.138cm}{1.494cm}}{\pgfqpoint{1.112cm}{1.468cm}}
\pgfpathcurveto{\pgfqpoint{1.087cm}{1.442cm}}{\pgfqpoint{1.072cm}{1.408cm}}{\pgfqpoint{1.072cm}{1.371cm}}
\pgfpathcurveto{\pgfqpoint{1.072cm}{1.335cm}}{\pgfqpoint{1.087cm}{1.3cm}}{\pgfqpoint{1.112cm}{1.274cm}}
\pgfpathcurveto{\pgfqpoint{1.138cm}{1.249cm}}{\pgfqpoint{1.172cm}{1.234cm}}{\pgfqpoint{1.209cm}{1.234cm}}
\pgfpathcurveto{\pgfqpoint{1.245cm}{1.234cm}}{\pgfqpoint{1.28cm}{1.249cm}}{\pgfqpoint{1.305cm}{1.274cm}}
\pgfpathcurveto{\pgfqpoint{1.331cm}{1.3cm}}{\pgfqpoint{1.345cm}{1.335cm}}{\pgfqpoint{1.345cm}{1.371cm}}
\pgfusepath{fill}
\begin{pgfscope}
\pgfsetdash{}{0cm}
\pgfsetlinewidth{0.818mm}
\pgfsetroundcap
\pgfsetmiterlimit{4.0}
\pgfpathmoveto{\pgfqpoint{0.682cm}{0.671cm}}
\pgfpathlineto{\pgfqpoint{0.682cm}{0.042cm}}
\pgfusepath{stroke}
\end{pgfscope}
\end{pgfscope}
\end{pgfscope}
\end{pgfscope}
\end{tikzpicture}}} + \phi + \psi)^2 + (- X^{\!\resizebox{0.6em}{!}{
\begin{tikzpicture}
\pgfpathmoveto{\pgfqpoint{0cm}{-0.035cm}}
\pgfpathlineto{\pgfqpoint{1.376cm}{-0.035cm}}
\pgfpathlineto{\pgfqpoint{1.376cm}{1.552cm}}
\pgfpathlineto{\pgfqpoint{0cm}{1.552cm}}
\pgfpathclose
\pgfusepath{clip}
\begin{pgfscope}
\begin{pgfscope}
\pgfpathmoveto{\pgfqpoint{0cm}{-0.035cm}}
\pgfpathlineto{\pgfqpoint{1.376cm}{-0.035cm}}
\pgfpathlineto{\pgfqpoint{1.376cm}{1.552cm}}
\pgfpathlineto{\pgfqpoint{0cm}{1.552cm}}
\pgfpathclose
\pgfusepath{clip}
\begin{pgfscope}
\begin{pgfscope}
\pgfsetdash{}{0cm}
\pgfsetlinewidth{0.818mm}
\pgfsetroundcap
\pgfsetroundjoin
\pgfsetmiterlimit{7.0}
\definecolor{eps2pgf_color}{gray}{0}\pgfsetstrokecolor{eps2pgf_color}\pgfsetfillcolor{eps2pgf_color}
\pgfpathmoveto{\pgfqpoint{0.117cm}{1.421cm}}
\pgfpathlineto{\pgfqpoint{0.682cm}{0.671cm}}
\pgfpathlineto{\pgfqpoint{1.246cm}{1.421cm}}
\pgfusepath{stroke}
\end{pgfscope}
\definecolor{eps2pgf_color}{gray}{0}\pgfsetstrokecolor{eps2pgf_color}\pgfsetfillcolor{eps2pgf_color}
\pgfpathmoveto{\pgfqpoint{0.273cm}{1.395cm}}
\pgfpathcurveto{\pgfqpoint{0.273cm}{1.432cm}}{\pgfqpoint{0.259cm}{1.467cm}}{\pgfqpoint{0.233cm}{1.492cm}}
\pgfpathcurveto{\pgfqpoint{0.207cm}{1.518cm}}{\pgfqpoint{0.173cm}{1.532cm}}{\pgfqpoint{0.137cm}{1.532cm}}
\pgfpathcurveto{\pgfqpoint{0.1cm}{1.532cm}}{\pgfqpoint{0.066cm}{1.518cm}}{\pgfqpoint{0.04cm}{1.492cm}}
\pgfpathcurveto{\pgfqpoint{0.014cm}{1.467cm}}{\pgfqpoint{0cm}{1.432cm}}{\pgfqpoint{0cm}{1.395cm}}
\pgfpathcurveto{\pgfqpoint{0cm}{1.359cm}}{\pgfqpoint{0.014cm}{1.324cm}}{\pgfqpoint{0.04cm}{1.299cm}}
\pgfpathcurveto{\pgfqpoint{0.066cm}{1.273cm}}{\pgfqpoint{0.1cm}{1.258cm}}{\pgfqpoint{0.137cm}{1.258cm}}
\pgfpathcurveto{\pgfqpoint{0.173cm}{1.258cm}}{\pgfqpoint{0.207cm}{1.273cm}}{\pgfqpoint{0.233cm}{1.299cm}}
\pgfpathcurveto{\pgfqpoint{0.259cm}{1.324cm}}{\pgfqpoint{0.273cm}{1.359cm}}{\pgfqpoint{0.273cm}{1.395cm}}
\pgfusepath{fill}
\begin{pgfscope}
\pgfsetdash{}{0cm}
\pgfsetlinewidth{0.818mm}
\pgfsetmiterlimit{7.0}
\pgfpathmoveto{\pgfqpoint{0.682cm}{0.671cm}}
\pgfpathlineto{\pgfqpoint{0.679cm}{1.418cm}}
\pgfusepath{stroke}
\end{pgfscope}
\pgfpathmoveto{\pgfqpoint{0.815cm}{1.399cm}}
\pgfpathcurveto{\pgfqpoint{0.815cm}{1.435cm}}{\pgfqpoint{0.801cm}{1.47cm}}{\pgfqpoint{0.775cm}{1.496cm}}
\pgfpathcurveto{\pgfqpoint{0.75cm}{1.521cm}}{\pgfqpoint{0.715cm}{1.536cm}}{\pgfqpoint{0.679cm}{1.536cm}}
\pgfpathcurveto{\pgfqpoint{0.643cm}{1.536cm}}{\pgfqpoint{0.608cm}{1.521cm}}{\pgfqpoint{0.582cm}{1.496cm}}
\pgfpathcurveto{\pgfqpoint{0.557cm}{1.47cm}}{\pgfqpoint{0.542cm}{1.435cm}}{\pgfqpoint{0.542cm}{1.399cm}}
\pgfpathcurveto{\pgfqpoint{0.542cm}{1.363cm}}{\pgfqpoint{0.557cm}{1.328cm}}{\pgfqpoint{0.582cm}{1.302cm}}
\pgfpathcurveto{\pgfqpoint{0.608cm}{1.276cm}}{\pgfqpoint{0.643cm}{1.262cm}}{\pgfqpoint{0.679cm}{1.262cm}}
\pgfpathcurveto{\pgfqpoint{0.715cm}{1.262cm}}{\pgfqpoint{0.75cm}{1.276cm}}{\pgfqpoint{0.775cm}{1.302cm}}
\pgfpathcurveto{\pgfqpoint{0.801cm}{1.328cm}}{\pgfqpoint{0.815cm}{1.363cm}}{\pgfqpoint{0.815cm}{1.399cm}}
\pgfusepath{fill}
\pgfpathmoveto{\pgfqpoint{1.345cm}{1.371cm}}
\pgfpathcurveto{\pgfqpoint{1.345cm}{1.408cm}}{\pgfqpoint{1.331cm}{1.442cm}}{\pgfqpoint{1.305cm}{1.468cm}}
\pgfpathcurveto{\pgfqpoint{1.28cm}{1.494cm}}{\pgfqpoint{1.245cm}{1.508cm}}{\pgfqpoint{1.209cm}{1.508cm}}
\pgfpathcurveto{\pgfqpoint{1.172cm}{1.508cm}}{\pgfqpoint{1.138cm}{1.494cm}}{\pgfqpoint{1.112cm}{1.468cm}}
\pgfpathcurveto{\pgfqpoint{1.087cm}{1.442cm}}{\pgfqpoint{1.072cm}{1.408cm}}{\pgfqpoint{1.072cm}{1.371cm}}
\pgfpathcurveto{\pgfqpoint{1.072cm}{1.335cm}}{\pgfqpoint{1.087cm}{1.3cm}}{\pgfqpoint{1.112cm}{1.274cm}}
\pgfpathcurveto{\pgfqpoint{1.138cm}{1.249cm}}{\pgfqpoint{1.172cm}{1.234cm}}{\pgfqpoint{1.209cm}{1.234cm}}
\pgfpathcurveto{\pgfqpoint{1.245cm}{1.234cm}}{\pgfqpoint{1.28cm}{1.249cm}}{\pgfqpoint{1.305cm}{1.274cm}}
\pgfpathcurveto{\pgfqpoint{1.331cm}{1.3cm}}{\pgfqpoint{1.345cm}{1.335cm}}{\pgfqpoint{1.345cm}{1.371cm}}
\pgfusepath{fill}
\begin{pgfscope}
\pgfsetdash{}{0cm}
\pgfsetlinewidth{0.818mm}
\pgfsetroundcap
\pgfsetmiterlimit{4.0}
\pgfpathmoveto{\pgfqpoint{0.682cm}{0.671cm}}
\pgfpathlineto{\pgfqpoint{0.682cm}{0.042cm}}
\pgfusepath{stroke}
\end{pgfscope}
\end{pgfscope}
\end{pgfscope}
\end{pgfscope}
\end{tikzpicture}}} +
      \phi + \psi)^3 + 3 b \varphi .
  \end{align*}
For notational simplicity, we chose  to write the above equation in this not rigorous  form -- with the infinite constant $b$ appearing -- instead of introducing the full decomposition with all the trees. Indeed,  we  are actually interested in a difference of the corresponding equations for two solutions $\varphi$ and $\tilde\varphi$ starting from the same initial condition. Thus the decomposition will simplify as the terms that do not depend on the solutions  cancel out.

  So if we denote by $\tilde{\varphi} = X - X^{\!\resizebox{0.6em}{!}{
\begin{tikzpicture}
\pgfpathmoveto{\pgfqpoint{0cm}{-0.035cm}}
\pgfpathlineto{\pgfqpoint{1.376cm}{-0.035cm}}
\pgfpathlineto{\pgfqpoint{1.376cm}{1.552cm}}
\pgfpathlineto{\pgfqpoint{0cm}{1.552cm}}
\pgfpathclose
\pgfusepath{clip}
\begin{pgfscope}
\begin{pgfscope}
\pgfpathmoveto{\pgfqpoint{0cm}{-0.035cm}}
\pgfpathlineto{\pgfqpoint{1.376cm}{-0.035cm}}
\pgfpathlineto{\pgfqpoint{1.376cm}{1.552cm}}
\pgfpathlineto{\pgfqpoint{0cm}{1.552cm}}
\pgfpathclose
\pgfusepath{clip}
\begin{pgfscope}
\begin{pgfscope}
\pgfsetdash{}{0cm}
\pgfsetlinewidth{0.818mm}
\pgfsetroundcap
\pgfsetroundjoin
\pgfsetmiterlimit{7.0}
\definecolor{eps2pgf_color}{gray}{0}\pgfsetstrokecolor{eps2pgf_color}\pgfsetfillcolor{eps2pgf_color}
\pgfpathmoveto{\pgfqpoint{0.117cm}{1.421cm}}
\pgfpathlineto{\pgfqpoint{0.682cm}{0.671cm}}
\pgfpathlineto{\pgfqpoint{1.246cm}{1.421cm}}
\pgfusepath{stroke}
\end{pgfscope}
\definecolor{eps2pgf_color}{gray}{0}\pgfsetstrokecolor{eps2pgf_color}\pgfsetfillcolor{eps2pgf_color}
\pgfpathmoveto{\pgfqpoint{0.273cm}{1.395cm}}
\pgfpathcurveto{\pgfqpoint{0.273cm}{1.432cm}}{\pgfqpoint{0.259cm}{1.467cm}}{\pgfqpoint{0.233cm}{1.492cm}}
\pgfpathcurveto{\pgfqpoint{0.207cm}{1.518cm}}{\pgfqpoint{0.173cm}{1.532cm}}{\pgfqpoint{0.137cm}{1.532cm}}
\pgfpathcurveto{\pgfqpoint{0.1cm}{1.532cm}}{\pgfqpoint{0.066cm}{1.518cm}}{\pgfqpoint{0.04cm}{1.492cm}}
\pgfpathcurveto{\pgfqpoint{0.014cm}{1.467cm}}{\pgfqpoint{0cm}{1.432cm}}{\pgfqpoint{0cm}{1.395cm}}
\pgfpathcurveto{\pgfqpoint{0cm}{1.359cm}}{\pgfqpoint{0.014cm}{1.324cm}}{\pgfqpoint{0.04cm}{1.299cm}}
\pgfpathcurveto{\pgfqpoint{0.066cm}{1.273cm}}{\pgfqpoint{0.1cm}{1.258cm}}{\pgfqpoint{0.137cm}{1.258cm}}
\pgfpathcurveto{\pgfqpoint{0.173cm}{1.258cm}}{\pgfqpoint{0.207cm}{1.273cm}}{\pgfqpoint{0.233cm}{1.299cm}}
\pgfpathcurveto{\pgfqpoint{0.259cm}{1.324cm}}{\pgfqpoint{0.273cm}{1.359cm}}{\pgfqpoint{0.273cm}{1.395cm}}
\pgfusepath{fill}
\begin{pgfscope}
\pgfsetdash{}{0cm}
\pgfsetlinewidth{0.818mm}
\pgfsetmiterlimit{7.0}
\pgfpathmoveto{\pgfqpoint{0.682cm}{0.671cm}}
\pgfpathlineto{\pgfqpoint{0.679cm}{1.418cm}}
\pgfusepath{stroke}
\end{pgfscope}
\pgfpathmoveto{\pgfqpoint{0.815cm}{1.399cm}}
\pgfpathcurveto{\pgfqpoint{0.815cm}{1.435cm}}{\pgfqpoint{0.801cm}{1.47cm}}{\pgfqpoint{0.775cm}{1.496cm}}
\pgfpathcurveto{\pgfqpoint{0.75cm}{1.521cm}}{\pgfqpoint{0.715cm}{1.536cm}}{\pgfqpoint{0.679cm}{1.536cm}}
\pgfpathcurveto{\pgfqpoint{0.643cm}{1.536cm}}{\pgfqpoint{0.608cm}{1.521cm}}{\pgfqpoint{0.582cm}{1.496cm}}
\pgfpathcurveto{\pgfqpoint{0.557cm}{1.47cm}}{\pgfqpoint{0.542cm}{1.435cm}}{\pgfqpoint{0.542cm}{1.399cm}}
\pgfpathcurveto{\pgfqpoint{0.542cm}{1.363cm}}{\pgfqpoint{0.557cm}{1.328cm}}{\pgfqpoint{0.582cm}{1.302cm}}
\pgfpathcurveto{\pgfqpoint{0.608cm}{1.276cm}}{\pgfqpoint{0.643cm}{1.262cm}}{\pgfqpoint{0.679cm}{1.262cm}}
\pgfpathcurveto{\pgfqpoint{0.715cm}{1.262cm}}{\pgfqpoint{0.75cm}{1.276cm}}{\pgfqpoint{0.775cm}{1.302cm}}
\pgfpathcurveto{\pgfqpoint{0.801cm}{1.328cm}}{\pgfqpoint{0.815cm}{1.363cm}}{\pgfqpoint{0.815cm}{1.399cm}}
\pgfusepath{fill}
\pgfpathmoveto{\pgfqpoint{1.345cm}{1.371cm}}
\pgfpathcurveto{\pgfqpoint{1.345cm}{1.408cm}}{\pgfqpoint{1.331cm}{1.442cm}}{\pgfqpoint{1.305cm}{1.468cm}}
\pgfpathcurveto{\pgfqpoint{1.28cm}{1.494cm}}{\pgfqpoint{1.245cm}{1.508cm}}{\pgfqpoint{1.209cm}{1.508cm}}
\pgfpathcurveto{\pgfqpoint{1.172cm}{1.508cm}}{\pgfqpoint{1.138cm}{1.494cm}}{\pgfqpoint{1.112cm}{1.468cm}}
\pgfpathcurveto{\pgfqpoint{1.087cm}{1.442cm}}{\pgfqpoint{1.072cm}{1.408cm}}{\pgfqpoint{1.072cm}{1.371cm}}
\pgfpathcurveto{\pgfqpoint{1.072cm}{1.335cm}}{\pgfqpoint{1.087cm}{1.3cm}}{\pgfqpoint{1.112cm}{1.274cm}}
\pgfpathcurveto{\pgfqpoint{1.138cm}{1.249cm}}{\pgfqpoint{1.172cm}{1.234cm}}{\pgfqpoint{1.209cm}{1.234cm}}
\pgfpathcurveto{\pgfqpoint{1.245cm}{1.234cm}}{\pgfqpoint{1.28cm}{1.249cm}}{\pgfqpoint{1.305cm}{1.274cm}}
\pgfpathcurveto{\pgfqpoint{1.331cm}{1.3cm}}{\pgfqpoint{1.345cm}{1.335cm}}{\pgfqpoint{1.345cm}{1.371cm}}
\pgfusepath{fill}
\begin{pgfscope}
\pgfsetdash{}{0cm}
\pgfsetlinewidth{0.818mm}
\pgfsetroundcap
\pgfsetmiterlimit{4.0}
\pgfpathmoveto{\pgfqpoint{0.682cm}{0.671cm}}
\pgfpathlineto{\pgfqpoint{0.682cm}{0.042cm}}
\pgfusepath{stroke}
\end{pgfscope}
\end{pgfscope}
\end{pgfscope}
\end{pgfscope}
\end{tikzpicture}}} + \tilde{\phi} +
  \tilde{\psi}$ another solution starting from the same initial condition and set $\zeta = \varphi - \tilde{\varphi}$
  and $\eta = \theta - \tilde{\theta} + \psi - \tilde{\psi}$, then we
  have
  \[ \zeta = \phi + \psi - \tilde{\phi} - \tilde{\psi} = - 3 \zeta \prec
     X^{\!\resizebox{0.6em}{!}{
\begin{tikzpicture}
\pgfpathmoveto{\pgfqpoint{0cm}{0cm}}
\pgfpathlineto{\pgfqpoint{1.376cm}{0cm}}
\pgfpathlineto{\pgfqpoint{1.376cm}{1.588cm}}
\pgfpathlineto{\pgfqpoint{0cm}{1.588cm}}
\pgfpathclose
\pgfusepath{clip}
\begin{pgfscope}
\begin{pgfscope}
\pgfpathmoveto{\pgfqpoint{0cm}{0cm}}
\pgfpathlineto{\pgfqpoint{1.376cm}{0cm}}
\pgfpathlineto{\pgfqpoint{1.376cm}{1.588cm}}
\pgfpathlineto{\pgfqpoint{0cm}{1.588cm}}
\pgfpathclose
\pgfusepath{clip}
\begin{pgfscope}
\begin{pgfscope}
\definecolor{eps2pgf_color}{gray}{0.976471}\pgfsetstrokecolor{eps2pgf_color}\pgfsetfillcolor{eps2pgf_color}
\pgfpathmoveto{\pgfqpoint{0cm}{0cm}}
\pgfpathlineto{\pgfqpoint{1.376cm}{0cm}}
\pgfpathlineto{\pgfqpoint{1.376cm}{1.588cm}}
\pgfpathlineto{\pgfqpoint{0cm}{1.588cm}}
\pgfpathclose
\pgfusepath{fill}
\end{pgfscope}
\begin{pgfscope}
\pgfsetdash{}{0cm}
\pgfsetlinewidth{0.818mm}
\pgfsetroundcap
\pgfsetroundjoin
\pgfsetmiterlimit{7.0}
\definecolor{eps2pgf_color}{gray}{0}\pgfsetstrokecolor{eps2pgf_color}\pgfsetfillcolor{eps2pgf_color}
\pgfpathmoveto{\pgfqpoint{0.117cm}{1.476cm}}
\pgfpathlineto{\pgfqpoint{0.682cm}{0.726cm}}
\pgfpathlineto{\pgfqpoint{1.246cm}{1.476cm}}
\pgfusepath{stroke}
\end{pgfscope}
\definecolor{eps2pgf_color}{gray}{0}\pgfsetstrokecolor{eps2pgf_color}\pgfsetfillcolor{eps2pgf_color}
\pgfpathmoveto{\pgfqpoint{0.273cm}{1.451cm}}
\pgfpathcurveto{\pgfqpoint{0.273cm}{1.487cm}}{\pgfqpoint{0.259cm}{1.522cm}}{\pgfqpoint{0.233cm}{1.547cm}}
\pgfpathcurveto{\pgfqpoint{0.207cm}{1.573cm}}{\pgfqpoint{0.173cm}{1.588cm}}{\pgfqpoint{0.137cm}{1.588cm}}
\pgfpathcurveto{\pgfqpoint{0.1cm}{1.588cm}}{\pgfqpoint{0.066cm}{1.573cm}}{\pgfqpoint{0.04cm}{1.547cm}}
\pgfpathcurveto{\pgfqpoint{0.014cm}{1.522cm}}{\pgfqpoint{0cm}{1.487cm}}{\pgfqpoint{0cm}{1.451cm}}
\pgfpathcurveto{\pgfqpoint{0cm}{1.414cm}}{\pgfqpoint{0.014cm}{1.379cm}}{\pgfqpoint{0.04cm}{1.354cm}}
\pgfpathcurveto{\pgfqpoint{0.066cm}{1.328cm}}{\pgfqpoint{0.1cm}{1.314cm}}{\pgfqpoint{0.137cm}{1.314cm}}
\pgfpathcurveto{\pgfqpoint{0.173cm}{1.314cm}}{\pgfqpoint{0.207cm}{1.328cm}}{\pgfqpoint{0.233cm}{1.354cm}}
\pgfpathcurveto{\pgfqpoint{0.259cm}{1.379cm}}{\pgfqpoint{0.273cm}{1.414cm}}{\pgfqpoint{0.273cm}{1.451cm}}
\pgfusepath{fill}
\pgfpathmoveto{\pgfqpoint{1.345cm}{1.426cm}}
\pgfpathcurveto{\pgfqpoint{1.345cm}{1.463cm}}{\pgfqpoint{1.331cm}{1.497cm}}{\pgfqpoint{1.305cm}{1.523cm}}
\pgfpathcurveto{\pgfqpoint{1.28cm}{1.549cm}}{\pgfqpoint{1.245cm}{1.563cm}}{\pgfqpoint{1.209cm}{1.563cm}}
\pgfpathcurveto{\pgfqpoint{1.172cm}{1.563cm}}{\pgfqpoint{1.138cm}{1.549cm}}{\pgfqpoint{1.112cm}{1.523cm}}
\pgfpathcurveto{\pgfqpoint{1.087cm}{1.497cm}}{\pgfqpoint{1.072cm}{1.463cm}}{\pgfqpoint{1.072cm}{1.426cm}}
\pgfpathcurveto{\pgfqpoint{1.072cm}{1.39cm}}{\pgfqpoint{1.087cm}{1.355cm}}{\pgfqpoint{1.112cm}{1.329cm}}
\pgfpathcurveto{\pgfqpoint{1.138cm}{1.304cm}}{\pgfqpoint{1.172cm}{1.289cm}}{\pgfqpoint{1.209cm}{1.289cm}}
\pgfpathcurveto{\pgfqpoint{1.245cm}{1.289cm}}{\pgfqpoint{1.28cm}{1.304cm}}{\pgfqpoint{1.305cm}{1.329cm}}
\pgfpathcurveto{\pgfqpoint{1.331cm}{1.355cm}}{\pgfqpoint{1.345cm}{1.39cm}}{\pgfqpoint{1.345cm}{1.426cm}}
\pgfusepath{fill}
\begin{pgfscope}
\pgfsetdash{}{0cm}
\pgfsetlinewidth{0.818mm}
\pgfsetroundcap
\pgfsetmiterlimit{4.0}
\pgfpathmoveto{\pgfqpoint{0.682cm}{0.726cm}}
\pgfpathlineto{\pgfqpoint{0.682cm}{0.097cm}}
\pgfusepath{stroke}
\end{pgfscope}
\end{pgfscope}
\end{pgfscope}
\end{pgfscope}
\end{tikzpicture}}} + \eta . \]
  In addition, it holds
  \[ 0 = \LL \zeta + 3 \llbracket X^2 \rrbracket \zeta + 3 b \zeta + \Upsilon
     \zeta, \qquad \zeta(0)=0,\]
  \[ 0 = \LL \eta + 3 \llbracket X^2 \rrbracket \preccurlyeq
     \zeta + 3 b \zeta - 3 [\LL, \zeta \prec] X^{\!\resizebox{0.6em}{!}{
\begin{tikzpicture}
\pgfpathmoveto{\pgfqpoint{0cm}{0cm}}
\pgfpathlineto{\pgfqpoint{1.376cm}{0cm}}
\pgfpathlineto{\pgfqpoint{1.376cm}{1.588cm}}
\pgfpathlineto{\pgfqpoint{0cm}{1.588cm}}
\pgfpathclose
\pgfusepath{clip}
\begin{pgfscope}
\begin{pgfscope}
\pgfpathmoveto{\pgfqpoint{0cm}{0cm}}
\pgfpathlineto{\pgfqpoint{1.376cm}{0cm}}
\pgfpathlineto{\pgfqpoint{1.376cm}{1.588cm}}
\pgfpathlineto{\pgfqpoint{0cm}{1.588cm}}
\pgfpathclose
\pgfusepath{clip}
\begin{pgfscope}
\begin{pgfscope}
\definecolor{eps2pgf_color}{gray}{0.976471}\pgfsetstrokecolor{eps2pgf_color}\pgfsetfillcolor{eps2pgf_color}
\pgfpathmoveto{\pgfqpoint{0cm}{0cm}}
\pgfpathlineto{\pgfqpoint{1.376cm}{0cm}}
\pgfpathlineto{\pgfqpoint{1.376cm}{1.588cm}}
\pgfpathlineto{\pgfqpoint{0cm}{1.588cm}}
\pgfpathclose
\pgfusepath{fill}
\end{pgfscope}
\begin{pgfscope}
\pgfsetdash{}{0cm}
\pgfsetlinewidth{0.818mm}
\pgfsetroundcap
\pgfsetroundjoin
\pgfsetmiterlimit{7.0}
\definecolor{eps2pgf_color}{gray}{0}\pgfsetstrokecolor{eps2pgf_color}\pgfsetfillcolor{eps2pgf_color}
\pgfpathmoveto{\pgfqpoint{0.117cm}{1.476cm}}
\pgfpathlineto{\pgfqpoint{0.682cm}{0.726cm}}
\pgfpathlineto{\pgfqpoint{1.246cm}{1.476cm}}
\pgfusepath{stroke}
\end{pgfscope}
\definecolor{eps2pgf_color}{gray}{0}\pgfsetstrokecolor{eps2pgf_color}\pgfsetfillcolor{eps2pgf_color}
\pgfpathmoveto{\pgfqpoint{0.273cm}{1.451cm}}
\pgfpathcurveto{\pgfqpoint{0.273cm}{1.487cm}}{\pgfqpoint{0.259cm}{1.522cm}}{\pgfqpoint{0.233cm}{1.547cm}}
\pgfpathcurveto{\pgfqpoint{0.207cm}{1.573cm}}{\pgfqpoint{0.173cm}{1.588cm}}{\pgfqpoint{0.137cm}{1.588cm}}
\pgfpathcurveto{\pgfqpoint{0.1cm}{1.588cm}}{\pgfqpoint{0.066cm}{1.573cm}}{\pgfqpoint{0.04cm}{1.547cm}}
\pgfpathcurveto{\pgfqpoint{0.014cm}{1.522cm}}{\pgfqpoint{0cm}{1.487cm}}{\pgfqpoint{0cm}{1.451cm}}
\pgfpathcurveto{\pgfqpoint{0cm}{1.414cm}}{\pgfqpoint{0.014cm}{1.379cm}}{\pgfqpoint{0.04cm}{1.354cm}}
\pgfpathcurveto{\pgfqpoint{0.066cm}{1.328cm}}{\pgfqpoint{0.1cm}{1.314cm}}{\pgfqpoint{0.137cm}{1.314cm}}
\pgfpathcurveto{\pgfqpoint{0.173cm}{1.314cm}}{\pgfqpoint{0.207cm}{1.328cm}}{\pgfqpoint{0.233cm}{1.354cm}}
\pgfpathcurveto{\pgfqpoint{0.259cm}{1.379cm}}{\pgfqpoint{0.273cm}{1.414cm}}{\pgfqpoint{0.273cm}{1.451cm}}
\pgfusepath{fill}
\pgfpathmoveto{\pgfqpoint{1.345cm}{1.426cm}}
\pgfpathcurveto{\pgfqpoint{1.345cm}{1.463cm}}{\pgfqpoint{1.331cm}{1.497cm}}{\pgfqpoint{1.305cm}{1.523cm}}
\pgfpathcurveto{\pgfqpoint{1.28cm}{1.549cm}}{\pgfqpoint{1.245cm}{1.563cm}}{\pgfqpoint{1.209cm}{1.563cm}}
\pgfpathcurveto{\pgfqpoint{1.172cm}{1.563cm}}{\pgfqpoint{1.138cm}{1.549cm}}{\pgfqpoint{1.112cm}{1.523cm}}
\pgfpathcurveto{\pgfqpoint{1.087cm}{1.497cm}}{\pgfqpoint{1.072cm}{1.463cm}}{\pgfqpoint{1.072cm}{1.426cm}}
\pgfpathcurveto{\pgfqpoint{1.072cm}{1.39cm}}{\pgfqpoint{1.087cm}{1.355cm}}{\pgfqpoint{1.112cm}{1.329cm}}
\pgfpathcurveto{\pgfqpoint{1.138cm}{1.304cm}}{\pgfqpoint{1.172cm}{1.289cm}}{\pgfqpoint{1.209cm}{1.289cm}}
\pgfpathcurveto{\pgfqpoint{1.245cm}{1.289cm}}{\pgfqpoint{1.28cm}{1.304cm}}{\pgfqpoint{1.305cm}{1.329cm}}
\pgfpathcurveto{\pgfqpoint{1.331cm}{1.355cm}}{\pgfqpoint{1.345cm}{1.39cm}}{\pgfqpoint{1.345cm}{1.426cm}}
\pgfusepath{fill}
\begin{pgfscope}
\pgfsetdash{}{0cm}
\pgfsetlinewidth{0.818mm}
\pgfsetroundcap
\pgfsetmiterlimit{4.0}
\pgfpathmoveto{\pgfqpoint{0.682cm}{0.726cm}}
\pgfpathlineto{\pgfqpoint{0.682cm}{0.097cm}}
\pgfusepath{stroke}
\end{pgfscope}
\end{pgfscope}
\end{pgfscope}
\end{pgfscope}
\end{tikzpicture}}} + \Upsilon
     \zeta,  \qquad \eta(0)=0,\]
  with
  \begin{align*}
      \Upsilon = & 3 X (- 2 X^{\!\resizebox{0.6em}{!}{
\begin{tikzpicture}
\pgfpathmoveto{\pgfqpoint{0cm}{-0.035cm}}
\pgfpathlineto{\pgfqpoint{1.376cm}{-0.035cm}}
\pgfpathlineto{\pgfqpoint{1.376cm}{1.552cm}}
\pgfpathlineto{\pgfqpoint{0cm}{1.552cm}}
\pgfpathclose
\pgfusepath{clip}
\begin{pgfscope}
\begin{pgfscope}
\pgfpathmoveto{\pgfqpoint{0cm}{-0.035cm}}
\pgfpathlineto{\pgfqpoint{1.376cm}{-0.035cm}}
\pgfpathlineto{\pgfqpoint{1.376cm}{1.552cm}}
\pgfpathlineto{\pgfqpoint{0cm}{1.552cm}}
\pgfpathclose
\pgfusepath{clip}
\begin{pgfscope}
\begin{pgfscope}
\pgfsetdash{}{0cm}
\pgfsetlinewidth{0.818mm}
\pgfsetroundcap
\pgfsetroundjoin
\pgfsetmiterlimit{7.0}
\definecolor{eps2pgf_color}{gray}{0}\pgfsetstrokecolor{eps2pgf_color}\pgfsetfillcolor{eps2pgf_color}
\pgfpathmoveto{\pgfqpoint{0.117cm}{1.421cm}}
\pgfpathlineto{\pgfqpoint{0.682cm}{0.671cm}}
\pgfpathlineto{\pgfqpoint{1.246cm}{1.421cm}}
\pgfusepath{stroke}
\end{pgfscope}
\definecolor{eps2pgf_color}{gray}{0}\pgfsetstrokecolor{eps2pgf_color}\pgfsetfillcolor{eps2pgf_color}
\pgfpathmoveto{\pgfqpoint{0.273cm}{1.395cm}}
\pgfpathcurveto{\pgfqpoint{0.273cm}{1.432cm}}{\pgfqpoint{0.259cm}{1.467cm}}{\pgfqpoint{0.233cm}{1.492cm}}
\pgfpathcurveto{\pgfqpoint{0.207cm}{1.518cm}}{\pgfqpoint{0.173cm}{1.532cm}}{\pgfqpoint{0.137cm}{1.532cm}}
\pgfpathcurveto{\pgfqpoint{0.1cm}{1.532cm}}{\pgfqpoint{0.066cm}{1.518cm}}{\pgfqpoint{0.04cm}{1.492cm}}
\pgfpathcurveto{\pgfqpoint{0.014cm}{1.467cm}}{\pgfqpoint{0cm}{1.432cm}}{\pgfqpoint{0cm}{1.395cm}}
\pgfpathcurveto{\pgfqpoint{0cm}{1.359cm}}{\pgfqpoint{0.014cm}{1.324cm}}{\pgfqpoint{0.04cm}{1.299cm}}
\pgfpathcurveto{\pgfqpoint{0.066cm}{1.273cm}}{\pgfqpoint{0.1cm}{1.258cm}}{\pgfqpoint{0.137cm}{1.258cm}}
\pgfpathcurveto{\pgfqpoint{0.173cm}{1.258cm}}{\pgfqpoint{0.207cm}{1.273cm}}{\pgfqpoint{0.233cm}{1.299cm}}
\pgfpathcurveto{\pgfqpoint{0.259cm}{1.324cm}}{\pgfqpoint{0.273cm}{1.359cm}}{\pgfqpoint{0.273cm}{1.395cm}}
\pgfusepath{fill}
\begin{pgfscope}
\pgfsetdash{}{0cm}
\pgfsetlinewidth{0.818mm}
\pgfsetmiterlimit{7.0}
\pgfpathmoveto{\pgfqpoint{0.682cm}{0.671cm}}
\pgfpathlineto{\pgfqpoint{0.679cm}{1.418cm}}
\pgfusepath{stroke}
\end{pgfscope}
\pgfpathmoveto{\pgfqpoint{0.815cm}{1.399cm}}
\pgfpathcurveto{\pgfqpoint{0.815cm}{1.435cm}}{\pgfqpoint{0.801cm}{1.47cm}}{\pgfqpoint{0.775cm}{1.496cm}}
\pgfpathcurveto{\pgfqpoint{0.75cm}{1.521cm}}{\pgfqpoint{0.715cm}{1.536cm}}{\pgfqpoint{0.679cm}{1.536cm}}
\pgfpathcurveto{\pgfqpoint{0.643cm}{1.536cm}}{\pgfqpoint{0.608cm}{1.521cm}}{\pgfqpoint{0.582cm}{1.496cm}}
\pgfpathcurveto{\pgfqpoint{0.557cm}{1.47cm}}{\pgfqpoint{0.542cm}{1.435cm}}{\pgfqpoint{0.542cm}{1.399cm}}
\pgfpathcurveto{\pgfqpoint{0.542cm}{1.363cm}}{\pgfqpoint{0.557cm}{1.328cm}}{\pgfqpoint{0.582cm}{1.302cm}}
\pgfpathcurveto{\pgfqpoint{0.608cm}{1.276cm}}{\pgfqpoint{0.643cm}{1.262cm}}{\pgfqpoint{0.679cm}{1.262cm}}
\pgfpathcurveto{\pgfqpoint{0.715cm}{1.262cm}}{\pgfqpoint{0.75cm}{1.276cm}}{\pgfqpoint{0.775cm}{1.302cm}}
\pgfpathcurveto{\pgfqpoint{0.801cm}{1.328cm}}{\pgfqpoint{0.815cm}{1.363cm}}{\pgfqpoint{0.815cm}{1.399cm}}
\pgfusepath{fill}
\pgfpathmoveto{\pgfqpoint{1.345cm}{1.371cm}}
\pgfpathcurveto{\pgfqpoint{1.345cm}{1.408cm}}{\pgfqpoint{1.331cm}{1.442cm}}{\pgfqpoint{1.305cm}{1.468cm}}
\pgfpathcurveto{\pgfqpoint{1.28cm}{1.494cm}}{\pgfqpoint{1.245cm}{1.508cm}}{\pgfqpoint{1.209cm}{1.508cm}}
\pgfpathcurveto{\pgfqpoint{1.172cm}{1.508cm}}{\pgfqpoint{1.138cm}{1.494cm}}{\pgfqpoint{1.112cm}{1.468cm}}
\pgfpathcurveto{\pgfqpoint{1.087cm}{1.442cm}}{\pgfqpoint{1.072cm}{1.408cm}}{\pgfqpoint{1.072cm}{1.371cm}}
\pgfpathcurveto{\pgfqpoint{1.072cm}{1.335cm}}{\pgfqpoint{1.087cm}{1.3cm}}{\pgfqpoint{1.112cm}{1.274cm}}
\pgfpathcurveto{\pgfqpoint{1.138cm}{1.249cm}}{\pgfqpoint{1.172cm}{1.234cm}}{\pgfqpoint{1.209cm}{1.234cm}}
\pgfpathcurveto{\pgfqpoint{1.245cm}{1.234cm}}{\pgfqpoint{1.28cm}{1.249cm}}{\pgfqpoint{1.305cm}{1.274cm}}
\pgfpathcurveto{\pgfqpoint{1.331cm}{1.3cm}}{\pgfqpoint{1.345cm}{1.335cm}}{\pgfqpoint{1.345cm}{1.371cm}}
\pgfusepath{fill}
\begin{pgfscope}
\pgfsetdash{}{0cm}
\pgfsetlinewidth{0.818mm}
\pgfsetroundcap
\pgfsetmiterlimit{4.0}
\pgfpathmoveto{\pgfqpoint{0.682cm}{0.671cm}}
\pgfpathlineto{\pgfqpoint{0.682cm}{0.042cm}}
\pgfusepath{stroke}
\end{pgfscope}
\end{pgfscope}
\end{pgfscope}
\end{pgfscope}
\end{tikzpicture}}} + \phi + \psi + \tilde{\phi} +
      \tilde{\psi}) \\
      & + [(- X^{\!\resizebox{0.6em}{!}{
\begin{tikzpicture}
\pgfpathmoveto{\pgfqpoint{0cm}{-0.035cm}}
\pgfpathlineto{\pgfqpoint{1.376cm}{-0.035cm}}
\pgfpathlineto{\pgfqpoint{1.376cm}{1.552cm}}
\pgfpathlineto{\pgfqpoint{0cm}{1.552cm}}
\pgfpathclose
\pgfusepath{clip}
\begin{pgfscope}
\begin{pgfscope}
\pgfpathmoveto{\pgfqpoint{0cm}{-0.035cm}}
\pgfpathlineto{\pgfqpoint{1.376cm}{-0.035cm}}
\pgfpathlineto{\pgfqpoint{1.376cm}{1.552cm}}
\pgfpathlineto{\pgfqpoint{0cm}{1.552cm}}
\pgfpathclose
\pgfusepath{clip}
\begin{pgfscope}
\begin{pgfscope}
\pgfsetdash{}{0cm}
\pgfsetlinewidth{0.818mm}
\pgfsetroundcap
\pgfsetroundjoin
\pgfsetmiterlimit{7.0}
\definecolor{eps2pgf_color}{gray}{0}\pgfsetstrokecolor{eps2pgf_color}\pgfsetfillcolor{eps2pgf_color}
\pgfpathmoveto{\pgfqpoint{0.117cm}{1.421cm}}
\pgfpathlineto{\pgfqpoint{0.682cm}{0.671cm}}
\pgfpathlineto{\pgfqpoint{1.246cm}{1.421cm}}
\pgfusepath{stroke}
\end{pgfscope}
\definecolor{eps2pgf_color}{gray}{0}\pgfsetstrokecolor{eps2pgf_color}\pgfsetfillcolor{eps2pgf_color}
\pgfpathmoveto{\pgfqpoint{0.273cm}{1.395cm}}
\pgfpathcurveto{\pgfqpoint{0.273cm}{1.432cm}}{\pgfqpoint{0.259cm}{1.467cm}}{\pgfqpoint{0.233cm}{1.492cm}}
\pgfpathcurveto{\pgfqpoint{0.207cm}{1.518cm}}{\pgfqpoint{0.173cm}{1.532cm}}{\pgfqpoint{0.137cm}{1.532cm}}
\pgfpathcurveto{\pgfqpoint{0.1cm}{1.532cm}}{\pgfqpoint{0.066cm}{1.518cm}}{\pgfqpoint{0.04cm}{1.492cm}}
\pgfpathcurveto{\pgfqpoint{0.014cm}{1.467cm}}{\pgfqpoint{0cm}{1.432cm}}{\pgfqpoint{0cm}{1.395cm}}
\pgfpathcurveto{\pgfqpoint{0cm}{1.359cm}}{\pgfqpoint{0.014cm}{1.324cm}}{\pgfqpoint{0.04cm}{1.299cm}}
\pgfpathcurveto{\pgfqpoint{0.066cm}{1.273cm}}{\pgfqpoint{0.1cm}{1.258cm}}{\pgfqpoint{0.137cm}{1.258cm}}
\pgfpathcurveto{\pgfqpoint{0.173cm}{1.258cm}}{\pgfqpoint{0.207cm}{1.273cm}}{\pgfqpoint{0.233cm}{1.299cm}}
\pgfpathcurveto{\pgfqpoint{0.259cm}{1.324cm}}{\pgfqpoint{0.273cm}{1.359cm}}{\pgfqpoint{0.273cm}{1.395cm}}
\pgfusepath{fill}
\begin{pgfscope}
\pgfsetdash{}{0cm}
\pgfsetlinewidth{0.818mm}
\pgfsetmiterlimit{7.0}
\pgfpathmoveto{\pgfqpoint{0.682cm}{0.671cm}}
\pgfpathlineto{\pgfqpoint{0.679cm}{1.418cm}}
\pgfusepath{stroke}
\end{pgfscope}
\pgfpathmoveto{\pgfqpoint{0.815cm}{1.399cm}}
\pgfpathcurveto{\pgfqpoint{0.815cm}{1.435cm}}{\pgfqpoint{0.801cm}{1.47cm}}{\pgfqpoint{0.775cm}{1.496cm}}
\pgfpathcurveto{\pgfqpoint{0.75cm}{1.521cm}}{\pgfqpoint{0.715cm}{1.536cm}}{\pgfqpoint{0.679cm}{1.536cm}}
\pgfpathcurveto{\pgfqpoint{0.643cm}{1.536cm}}{\pgfqpoint{0.608cm}{1.521cm}}{\pgfqpoint{0.582cm}{1.496cm}}
\pgfpathcurveto{\pgfqpoint{0.557cm}{1.47cm}}{\pgfqpoint{0.542cm}{1.435cm}}{\pgfqpoint{0.542cm}{1.399cm}}
\pgfpathcurveto{\pgfqpoint{0.542cm}{1.363cm}}{\pgfqpoint{0.557cm}{1.328cm}}{\pgfqpoint{0.582cm}{1.302cm}}
\pgfpathcurveto{\pgfqpoint{0.608cm}{1.276cm}}{\pgfqpoint{0.643cm}{1.262cm}}{\pgfqpoint{0.679cm}{1.262cm}}
\pgfpathcurveto{\pgfqpoint{0.715cm}{1.262cm}}{\pgfqpoint{0.75cm}{1.276cm}}{\pgfqpoint{0.775cm}{1.302cm}}
\pgfpathcurveto{\pgfqpoint{0.801cm}{1.328cm}}{\pgfqpoint{0.815cm}{1.363cm}}{\pgfqpoint{0.815cm}{1.399cm}}
\pgfusepath{fill}
\pgfpathmoveto{\pgfqpoint{1.345cm}{1.371cm}}
\pgfpathcurveto{\pgfqpoint{1.345cm}{1.408cm}}{\pgfqpoint{1.331cm}{1.442cm}}{\pgfqpoint{1.305cm}{1.468cm}}
\pgfpathcurveto{\pgfqpoint{1.28cm}{1.494cm}}{\pgfqpoint{1.245cm}{1.508cm}}{\pgfqpoint{1.209cm}{1.508cm}}
\pgfpathcurveto{\pgfqpoint{1.172cm}{1.508cm}}{\pgfqpoint{1.138cm}{1.494cm}}{\pgfqpoint{1.112cm}{1.468cm}}
\pgfpathcurveto{\pgfqpoint{1.087cm}{1.442cm}}{\pgfqpoint{1.072cm}{1.408cm}}{\pgfqpoint{1.072cm}{1.371cm}}
\pgfpathcurveto{\pgfqpoint{1.072cm}{1.335cm}}{\pgfqpoint{1.087cm}{1.3cm}}{\pgfqpoint{1.112cm}{1.274cm}}
\pgfpathcurveto{\pgfqpoint{1.138cm}{1.249cm}}{\pgfqpoint{1.172cm}{1.234cm}}{\pgfqpoint{1.209cm}{1.234cm}}
\pgfpathcurveto{\pgfqpoint{1.245cm}{1.234cm}}{\pgfqpoint{1.28cm}{1.249cm}}{\pgfqpoint{1.305cm}{1.274cm}}
\pgfpathcurveto{\pgfqpoint{1.331cm}{1.3cm}}{\pgfqpoint{1.345cm}{1.335cm}}{\pgfqpoint{1.345cm}{1.371cm}}
\pgfusepath{fill}
\begin{pgfscope}
\pgfsetdash{}{0cm}
\pgfsetlinewidth{0.818mm}
\pgfsetroundcap
\pgfsetmiterlimit{4.0}
\pgfpathmoveto{\pgfqpoint{0.682cm}{0.671cm}}
\pgfpathlineto{\pgfqpoint{0.682cm}{0.042cm}}
\pgfusepath{stroke}
\end{pgfscope}
\end{pgfscope}
\end{pgfscope}
\end{pgfscope}
\end{tikzpicture}}} + \phi + \psi)^2 + (- X^{\!\resizebox{0.6em}{!}{
\begin{tikzpicture}
\pgfpathmoveto{\pgfqpoint{0cm}{-0.035cm}}
\pgfpathlineto{\pgfqpoint{1.376cm}{-0.035cm}}
\pgfpathlineto{\pgfqpoint{1.376cm}{1.552cm}}
\pgfpathlineto{\pgfqpoint{0cm}{1.552cm}}
\pgfpathclose
\pgfusepath{clip}
\begin{pgfscope}
\begin{pgfscope}
\pgfpathmoveto{\pgfqpoint{0cm}{-0.035cm}}
\pgfpathlineto{\pgfqpoint{1.376cm}{-0.035cm}}
\pgfpathlineto{\pgfqpoint{1.376cm}{1.552cm}}
\pgfpathlineto{\pgfqpoint{0cm}{1.552cm}}
\pgfpathclose
\pgfusepath{clip}
\begin{pgfscope}
\begin{pgfscope}
\pgfsetdash{}{0cm}
\pgfsetlinewidth{0.818mm}
\pgfsetroundcap
\pgfsetroundjoin
\pgfsetmiterlimit{7.0}
\definecolor{eps2pgf_color}{gray}{0}\pgfsetstrokecolor{eps2pgf_color}\pgfsetfillcolor{eps2pgf_color}
\pgfpathmoveto{\pgfqpoint{0.117cm}{1.421cm}}
\pgfpathlineto{\pgfqpoint{0.682cm}{0.671cm}}
\pgfpathlineto{\pgfqpoint{1.246cm}{1.421cm}}
\pgfusepath{stroke}
\end{pgfscope}
\definecolor{eps2pgf_color}{gray}{0}\pgfsetstrokecolor{eps2pgf_color}\pgfsetfillcolor{eps2pgf_color}
\pgfpathmoveto{\pgfqpoint{0.273cm}{1.395cm}}
\pgfpathcurveto{\pgfqpoint{0.273cm}{1.432cm}}{\pgfqpoint{0.259cm}{1.467cm}}{\pgfqpoint{0.233cm}{1.492cm}}
\pgfpathcurveto{\pgfqpoint{0.207cm}{1.518cm}}{\pgfqpoint{0.173cm}{1.532cm}}{\pgfqpoint{0.137cm}{1.532cm}}
\pgfpathcurveto{\pgfqpoint{0.1cm}{1.532cm}}{\pgfqpoint{0.066cm}{1.518cm}}{\pgfqpoint{0.04cm}{1.492cm}}
\pgfpathcurveto{\pgfqpoint{0.014cm}{1.467cm}}{\pgfqpoint{0cm}{1.432cm}}{\pgfqpoint{0cm}{1.395cm}}
\pgfpathcurveto{\pgfqpoint{0cm}{1.359cm}}{\pgfqpoint{0.014cm}{1.324cm}}{\pgfqpoint{0.04cm}{1.299cm}}
\pgfpathcurveto{\pgfqpoint{0.066cm}{1.273cm}}{\pgfqpoint{0.1cm}{1.258cm}}{\pgfqpoint{0.137cm}{1.258cm}}
\pgfpathcurveto{\pgfqpoint{0.173cm}{1.258cm}}{\pgfqpoint{0.207cm}{1.273cm}}{\pgfqpoint{0.233cm}{1.299cm}}
\pgfpathcurveto{\pgfqpoint{0.259cm}{1.324cm}}{\pgfqpoint{0.273cm}{1.359cm}}{\pgfqpoint{0.273cm}{1.395cm}}
\pgfusepath{fill}
\begin{pgfscope}
\pgfsetdash{}{0cm}
\pgfsetlinewidth{0.818mm}
\pgfsetmiterlimit{7.0}
\pgfpathmoveto{\pgfqpoint{0.682cm}{0.671cm}}
\pgfpathlineto{\pgfqpoint{0.679cm}{1.418cm}}
\pgfusepath{stroke}
\end{pgfscope}
\pgfpathmoveto{\pgfqpoint{0.815cm}{1.399cm}}
\pgfpathcurveto{\pgfqpoint{0.815cm}{1.435cm}}{\pgfqpoint{0.801cm}{1.47cm}}{\pgfqpoint{0.775cm}{1.496cm}}
\pgfpathcurveto{\pgfqpoint{0.75cm}{1.521cm}}{\pgfqpoint{0.715cm}{1.536cm}}{\pgfqpoint{0.679cm}{1.536cm}}
\pgfpathcurveto{\pgfqpoint{0.643cm}{1.536cm}}{\pgfqpoint{0.608cm}{1.521cm}}{\pgfqpoint{0.582cm}{1.496cm}}
\pgfpathcurveto{\pgfqpoint{0.557cm}{1.47cm}}{\pgfqpoint{0.542cm}{1.435cm}}{\pgfqpoint{0.542cm}{1.399cm}}
\pgfpathcurveto{\pgfqpoint{0.542cm}{1.363cm}}{\pgfqpoint{0.557cm}{1.328cm}}{\pgfqpoint{0.582cm}{1.302cm}}
\pgfpathcurveto{\pgfqpoint{0.608cm}{1.276cm}}{\pgfqpoint{0.643cm}{1.262cm}}{\pgfqpoint{0.679cm}{1.262cm}}
\pgfpathcurveto{\pgfqpoint{0.715cm}{1.262cm}}{\pgfqpoint{0.75cm}{1.276cm}}{\pgfqpoint{0.775cm}{1.302cm}}
\pgfpathcurveto{\pgfqpoint{0.801cm}{1.328cm}}{\pgfqpoint{0.815cm}{1.363cm}}{\pgfqpoint{0.815cm}{1.399cm}}
\pgfusepath{fill}
\pgfpathmoveto{\pgfqpoint{1.345cm}{1.371cm}}
\pgfpathcurveto{\pgfqpoint{1.345cm}{1.408cm}}{\pgfqpoint{1.331cm}{1.442cm}}{\pgfqpoint{1.305cm}{1.468cm}}
\pgfpathcurveto{\pgfqpoint{1.28cm}{1.494cm}}{\pgfqpoint{1.245cm}{1.508cm}}{\pgfqpoint{1.209cm}{1.508cm}}
\pgfpathcurveto{\pgfqpoint{1.172cm}{1.508cm}}{\pgfqpoint{1.138cm}{1.494cm}}{\pgfqpoint{1.112cm}{1.468cm}}
\pgfpathcurveto{\pgfqpoint{1.087cm}{1.442cm}}{\pgfqpoint{1.072cm}{1.408cm}}{\pgfqpoint{1.072cm}{1.371cm}}
\pgfpathcurveto{\pgfqpoint{1.072cm}{1.335cm}}{\pgfqpoint{1.087cm}{1.3cm}}{\pgfqpoint{1.112cm}{1.274cm}}
\pgfpathcurveto{\pgfqpoint{1.138cm}{1.249cm}}{\pgfqpoint{1.172cm}{1.234cm}}{\pgfqpoint{1.209cm}{1.234cm}}
\pgfpathcurveto{\pgfqpoint{1.245cm}{1.234cm}}{\pgfqpoint{1.28cm}{1.249cm}}{\pgfqpoint{1.305cm}{1.274cm}}
\pgfpathcurveto{\pgfqpoint{1.331cm}{1.3cm}}{\pgfqpoint{1.345cm}{1.335cm}}{\pgfqpoint{1.345cm}{1.371cm}}
\pgfusepath{fill}
\begin{pgfscope}
\pgfsetdash{}{0cm}
\pgfsetlinewidth{0.818mm}
\pgfsetroundcap
\pgfsetmiterlimit{4.0}
\pgfpathmoveto{\pgfqpoint{0.682cm}{0.671cm}}
\pgfpathlineto{\pgfqpoint{0.682cm}{0.042cm}}
\pgfusepath{stroke}
\end{pgfscope}
\end{pgfscope}
\end{pgfscope}
\end{pgfscope}
\end{tikzpicture}}} + \phi +
      \psi) (- X^{\!\resizebox{0.6em}{!}{
\begin{tikzpicture}
\pgfpathmoveto{\pgfqpoint{0cm}{-0.035cm}}
\pgfpathlineto{\pgfqpoint{1.376cm}{-0.035cm}}
\pgfpathlineto{\pgfqpoint{1.376cm}{1.552cm}}
\pgfpathlineto{\pgfqpoint{0cm}{1.552cm}}
\pgfpathclose
\pgfusepath{clip}
\begin{pgfscope}
\begin{pgfscope}
\pgfpathmoveto{\pgfqpoint{0cm}{-0.035cm}}
\pgfpathlineto{\pgfqpoint{1.376cm}{-0.035cm}}
\pgfpathlineto{\pgfqpoint{1.376cm}{1.552cm}}
\pgfpathlineto{\pgfqpoint{0cm}{1.552cm}}
\pgfpathclose
\pgfusepath{clip}
\begin{pgfscope}
\begin{pgfscope}
\pgfsetdash{}{0cm}
\pgfsetlinewidth{0.818mm}
\pgfsetroundcap
\pgfsetroundjoin
\pgfsetmiterlimit{7.0}
\definecolor{eps2pgf_color}{gray}{0}\pgfsetstrokecolor{eps2pgf_color}\pgfsetfillcolor{eps2pgf_color}
\pgfpathmoveto{\pgfqpoint{0.117cm}{1.421cm}}
\pgfpathlineto{\pgfqpoint{0.682cm}{0.671cm}}
\pgfpathlineto{\pgfqpoint{1.246cm}{1.421cm}}
\pgfusepath{stroke}
\end{pgfscope}
\definecolor{eps2pgf_color}{gray}{0}\pgfsetstrokecolor{eps2pgf_color}\pgfsetfillcolor{eps2pgf_color}
\pgfpathmoveto{\pgfqpoint{0.273cm}{1.395cm}}
\pgfpathcurveto{\pgfqpoint{0.273cm}{1.432cm}}{\pgfqpoint{0.259cm}{1.467cm}}{\pgfqpoint{0.233cm}{1.492cm}}
\pgfpathcurveto{\pgfqpoint{0.207cm}{1.518cm}}{\pgfqpoint{0.173cm}{1.532cm}}{\pgfqpoint{0.137cm}{1.532cm}}
\pgfpathcurveto{\pgfqpoint{0.1cm}{1.532cm}}{\pgfqpoint{0.066cm}{1.518cm}}{\pgfqpoint{0.04cm}{1.492cm}}
\pgfpathcurveto{\pgfqpoint{0.014cm}{1.467cm}}{\pgfqpoint{0cm}{1.432cm}}{\pgfqpoint{0cm}{1.395cm}}
\pgfpathcurveto{\pgfqpoint{0cm}{1.359cm}}{\pgfqpoint{0.014cm}{1.324cm}}{\pgfqpoint{0.04cm}{1.299cm}}
\pgfpathcurveto{\pgfqpoint{0.066cm}{1.273cm}}{\pgfqpoint{0.1cm}{1.258cm}}{\pgfqpoint{0.137cm}{1.258cm}}
\pgfpathcurveto{\pgfqpoint{0.173cm}{1.258cm}}{\pgfqpoint{0.207cm}{1.273cm}}{\pgfqpoint{0.233cm}{1.299cm}}
\pgfpathcurveto{\pgfqpoint{0.259cm}{1.324cm}}{\pgfqpoint{0.273cm}{1.359cm}}{\pgfqpoint{0.273cm}{1.395cm}}
\pgfusepath{fill}
\begin{pgfscope}
\pgfsetdash{}{0cm}
\pgfsetlinewidth{0.818mm}
\pgfsetmiterlimit{7.0}
\pgfpathmoveto{\pgfqpoint{0.682cm}{0.671cm}}
\pgfpathlineto{\pgfqpoint{0.679cm}{1.418cm}}
\pgfusepath{stroke}
\end{pgfscope}
\pgfpathmoveto{\pgfqpoint{0.815cm}{1.399cm}}
\pgfpathcurveto{\pgfqpoint{0.815cm}{1.435cm}}{\pgfqpoint{0.801cm}{1.47cm}}{\pgfqpoint{0.775cm}{1.496cm}}
\pgfpathcurveto{\pgfqpoint{0.75cm}{1.521cm}}{\pgfqpoint{0.715cm}{1.536cm}}{\pgfqpoint{0.679cm}{1.536cm}}
\pgfpathcurveto{\pgfqpoint{0.643cm}{1.536cm}}{\pgfqpoint{0.608cm}{1.521cm}}{\pgfqpoint{0.582cm}{1.496cm}}
\pgfpathcurveto{\pgfqpoint{0.557cm}{1.47cm}}{\pgfqpoint{0.542cm}{1.435cm}}{\pgfqpoint{0.542cm}{1.399cm}}
\pgfpathcurveto{\pgfqpoint{0.542cm}{1.363cm}}{\pgfqpoint{0.557cm}{1.328cm}}{\pgfqpoint{0.582cm}{1.302cm}}
\pgfpathcurveto{\pgfqpoint{0.608cm}{1.276cm}}{\pgfqpoint{0.643cm}{1.262cm}}{\pgfqpoint{0.679cm}{1.262cm}}
\pgfpathcurveto{\pgfqpoint{0.715cm}{1.262cm}}{\pgfqpoint{0.75cm}{1.276cm}}{\pgfqpoint{0.775cm}{1.302cm}}
\pgfpathcurveto{\pgfqpoint{0.801cm}{1.328cm}}{\pgfqpoint{0.815cm}{1.363cm}}{\pgfqpoint{0.815cm}{1.399cm}}
\pgfusepath{fill}
\pgfpathmoveto{\pgfqpoint{1.345cm}{1.371cm}}
\pgfpathcurveto{\pgfqpoint{1.345cm}{1.408cm}}{\pgfqpoint{1.331cm}{1.442cm}}{\pgfqpoint{1.305cm}{1.468cm}}
\pgfpathcurveto{\pgfqpoint{1.28cm}{1.494cm}}{\pgfqpoint{1.245cm}{1.508cm}}{\pgfqpoint{1.209cm}{1.508cm}}
\pgfpathcurveto{\pgfqpoint{1.172cm}{1.508cm}}{\pgfqpoint{1.138cm}{1.494cm}}{\pgfqpoint{1.112cm}{1.468cm}}
\pgfpathcurveto{\pgfqpoint{1.087cm}{1.442cm}}{\pgfqpoint{1.072cm}{1.408cm}}{\pgfqpoint{1.072cm}{1.371cm}}
\pgfpathcurveto{\pgfqpoint{1.072cm}{1.335cm}}{\pgfqpoint{1.087cm}{1.3cm}}{\pgfqpoint{1.112cm}{1.274cm}}
\pgfpathcurveto{\pgfqpoint{1.138cm}{1.249cm}}{\pgfqpoint{1.172cm}{1.234cm}}{\pgfqpoint{1.209cm}{1.234cm}}
\pgfpathcurveto{\pgfqpoint{1.245cm}{1.234cm}}{\pgfqpoint{1.28cm}{1.249cm}}{\pgfqpoint{1.305cm}{1.274cm}}
\pgfpathcurveto{\pgfqpoint{1.331cm}{1.3cm}}{\pgfqpoint{1.345cm}{1.335cm}}{\pgfqpoint{1.345cm}{1.371cm}}
\pgfusepath{fill}
\begin{pgfscope}
\pgfsetdash{}{0cm}
\pgfsetlinewidth{0.818mm}
\pgfsetroundcap
\pgfsetmiterlimit{4.0}
\pgfpathmoveto{\pgfqpoint{0.682cm}{0.671cm}}
\pgfpathlineto{\pgfqpoint{0.682cm}{0.042cm}}
\pgfusepath{stroke}
\end{pgfscope}
\end{pgfscope}
\end{pgfscope}
\end{pgfscope}
\end{tikzpicture}}} + \tilde{\phi} + \tilde{\psi}) + (- X^{\!\resizebox{0.6em}{!}{
\begin{tikzpicture}
\pgfpathmoveto{\pgfqpoint{0cm}{-0.035cm}}
\pgfpathlineto{\pgfqpoint{1.376cm}{-0.035cm}}
\pgfpathlineto{\pgfqpoint{1.376cm}{1.552cm}}
\pgfpathlineto{\pgfqpoint{0cm}{1.552cm}}
\pgfpathclose
\pgfusepath{clip}
\begin{pgfscope}
\begin{pgfscope}
\pgfpathmoveto{\pgfqpoint{0cm}{-0.035cm}}
\pgfpathlineto{\pgfqpoint{1.376cm}{-0.035cm}}
\pgfpathlineto{\pgfqpoint{1.376cm}{1.552cm}}
\pgfpathlineto{\pgfqpoint{0cm}{1.552cm}}
\pgfpathclose
\pgfusepath{clip}
\begin{pgfscope}
\begin{pgfscope}
\pgfsetdash{}{0cm}
\pgfsetlinewidth{0.818mm}
\pgfsetroundcap
\pgfsetroundjoin
\pgfsetmiterlimit{7.0}
\definecolor{eps2pgf_color}{gray}{0}\pgfsetstrokecolor{eps2pgf_color}\pgfsetfillcolor{eps2pgf_color}
\pgfpathmoveto{\pgfqpoint{0.117cm}{1.421cm}}
\pgfpathlineto{\pgfqpoint{0.682cm}{0.671cm}}
\pgfpathlineto{\pgfqpoint{1.246cm}{1.421cm}}
\pgfusepath{stroke}
\end{pgfscope}
\definecolor{eps2pgf_color}{gray}{0}\pgfsetstrokecolor{eps2pgf_color}\pgfsetfillcolor{eps2pgf_color}
\pgfpathmoveto{\pgfqpoint{0.273cm}{1.395cm}}
\pgfpathcurveto{\pgfqpoint{0.273cm}{1.432cm}}{\pgfqpoint{0.259cm}{1.467cm}}{\pgfqpoint{0.233cm}{1.492cm}}
\pgfpathcurveto{\pgfqpoint{0.207cm}{1.518cm}}{\pgfqpoint{0.173cm}{1.532cm}}{\pgfqpoint{0.137cm}{1.532cm}}
\pgfpathcurveto{\pgfqpoint{0.1cm}{1.532cm}}{\pgfqpoint{0.066cm}{1.518cm}}{\pgfqpoint{0.04cm}{1.492cm}}
\pgfpathcurveto{\pgfqpoint{0.014cm}{1.467cm}}{\pgfqpoint{0cm}{1.432cm}}{\pgfqpoint{0cm}{1.395cm}}
\pgfpathcurveto{\pgfqpoint{0cm}{1.359cm}}{\pgfqpoint{0.014cm}{1.324cm}}{\pgfqpoint{0.04cm}{1.299cm}}
\pgfpathcurveto{\pgfqpoint{0.066cm}{1.273cm}}{\pgfqpoint{0.1cm}{1.258cm}}{\pgfqpoint{0.137cm}{1.258cm}}
\pgfpathcurveto{\pgfqpoint{0.173cm}{1.258cm}}{\pgfqpoint{0.207cm}{1.273cm}}{\pgfqpoint{0.233cm}{1.299cm}}
\pgfpathcurveto{\pgfqpoint{0.259cm}{1.324cm}}{\pgfqpoint{0.273cm}{1.359cm}}{\pgfqpoint{0.273cm}{1.395cm}}
\pgfusepath{fill}
\begin{pgfscope}
\pgfsetdash{}{0cm}
\pgfsetlinewidth{0.818mm}
\pgfsetmiterlimit{7.0}
\pgfpathmoveto{\pgfqpoint{0.682cm}{0.671cm}}
\pgfpathlineto{\pgfqpoint{0.679cm}{1.418cm}}
\pgfusepath{stroke}
\end{pgfscope}
\pgfpathmoveto{\pgfqpoint{0.815cm}{1.399cm}}
\pgfpathcurveto{\pgfqpoint{0.815cm}{1.435cm}}{\pgfqpoint{0.801cm}{1.47cm}}{\pgfqpoint{0.775cm}{1.496cm}}
\pgfpathcurveto{\pgfqpoint{0.75cm}{1.521cm}}{\pgfqpoint{0.715cm}{1.536cm}}{\pgfqpoint{0.679cm}{1.536cm}}
\pgfpathcurveto{\pgfqpoint{0.643cm}{1.536cm}}{\pgfqpoint{0.608cm}{1.521cm}}{\pgfqpoint{0.582cm}{1.496cm}}
\pgfpathcurveto{\pgfqpoint{0.557cm}{1.47cm}}{\pgfqpoint{0.542cm}{1.435cm}}{\pgfqpoint{0.542cm}{1.399cm}}
\pgfpathcurveto{\pgfqpoint{0.542cm}{1.363cm}}{\pgfqpoint{0.557cm}{1.328cm}}{\pgfqpoint{0.582cm}{1.302cm}}
\pgfpathcurveto{\pgfqpoint{0.608cm}{1.276cm}}{\pgfqpoint{0.643cm}{1.262cm}}{\pgfqpoint{0.679cm}{1.262cm}}
\pgfpathcurveto{\pgfqpoint{0.715cm}{1.262cm}}{\pgfqpoint{0.75cm}{1.276cm}}{\pgfqpoint{0.775cm}{1.302cm}}
\pgfpathcurveto{\pgfqpoint{0.801cm}{1.328cm}}{\pgfqpoint{0.815cm}{1.363cm}}{\pgfqpoint{0.815cm}{1.399cm}}
\pgfusepath{fill}
\pgfpathmoveto{\pgfqpoint{1.345cm}{1.371cm}}
\pgfpathcurveto{\pgfqpoint{1.345cm}{1.408cm}}{\pgfqpoint{1.331cm}{1.442cm}}{\pgfqpoint{1.305cm}{1.468cm}}
\pgfpathcurveto{\pgfqpoint{1.28cm}{1.494cm}}{\pgfqpoint{1.245cm}{1.508cm}}{\pgfqpoint{1.209cm}{1.508cm}}
\pgfpathcurveto{\pgfqpoint{1.172cm}{1.508cm}}{\pgfqpoint{1.138cm}{1.494cm}}{\pgfqpoint{1.112cm}{1.468cm}}
\pgfpathcurveto{\pgfqpoint{1.087cm}{1.442cm}}{\pgfqpoint{1.072cm}{1.408cm}}{\pgfqpoint{1.072cm}{1.371cm}}
\pgfpathcurveto{\pgfqpoint{1.072cm}{1.335cm}}{\pgfqpoint{1.087cm}{1.3cm}}{\pgfqpoint{1.112cm}{1.274cm}}
\pgfpathcurveto{\pgfqpoint{1.138cm}{1.249cm}}{\pgfqpoint{1.172cm}{1.234cm}}{\pgfqpoint{1.209cm}{1.234cm}}
\pgfpathcurveto{\pgfqpoint{1.245cm}{1.234cm}}{\pgfqpoint{1.28cm}{1.249cm}}{\pgfqpoint{1.305cm}{1.274cm}}
\pgfpathcurveto{\pgfqpoint{1.331cm}{1.3cm}}{\pgfqpoint{1.345cm}{1.335cm}}{\pgfqpoint{1.345cm}{1.371cm}}
\pgfusepath{fill}
\begin{pgfscope}
\pgfsetdash{}{0cm}
\pgfsetlinewidth{0.818mm}
\pgfsetroundcap
\pgfsetmiterlimit{4.0}
\pgfpathmoveto{\pgfqpoint{0.682cm}{0.671cm}}
\pgfpathlineto{\pgfqpoint{0.682cm}{0.042cm}}
\pgfusepath{stroke}
\end{pgfscope}
\end{pgfscope}
\end{pgfscope}
\end{pgfscope}
\end{tikzpicture}}}
      + \tilde{\phi} + \tilde{\psi})^2] .
  \end{align*}
  From the estimates of the solutions from Theorem \ref{thm:ex3d} we obtain $\Upsilon \in \CC^{-1/2- \kappa}
  (\rho^{\sigma})$. However, we stress that the term $3 \llbracket X^2
  \rrbracket \circ \zeta + 3 b \zeta$ is to be understood in the following sense
  \[ 3 \llbracket X^2 \rrbracket \circ \zeta + 3 b \zeta = - 9 \llbracket X^2
     \rrbracket \circ (\zeta \prec X^{\!\resizebox{0.6em}{!}{
\begin{tikzpicture}
\pgfpathmoveto{\pgfqpoint{0cm}{0cm}}
\pgfpathlineto{\pgfqpoint{1.376cm}{0cm}}
\pgfpathlineto{\pgfqpoint{1.376cm}{1.588cm}}
\pgfpathlineto{\pgfqpoint{0cm}{1.588cm}}
\pgfpathclose
\pgfusepath{clip}
\begin{pgfscope}
\begin{pgfscope}
\pgfpathmoveto{\pgfqpoint{0cm}{0cm}}
\pgfpathlineto{\pgfqpoint{1.376cm}{0cm}}
\pgfpathlineto{\pgfqpoint{1.376cm}{1.588cm}}
\pgfpathlineto{\pgfqpoint{0cm}{1.588cm}}
\pgfpathclose
\pgfusepath{clip}
\begin{pgfscope}
\begin{pgfscope}
\definecolor{eps2pgf_color}{gray}{0.976471}\pgfsetstrokecolor{eps2pgf_color}\pgfsetfillcolor{eps2pgf_color}
\pgfpathmoveto{\pgfqpoint{0cm}{0cm}}
\pgfpathlineto{\pgfqpoint{1.376cm}{0cm}}
\pgfpathlineto{\pgfqpoint{1.376cm}{1.588cm}}
\pgfpathlineto{\pgfqpoint{0cm}{1.588cm}}
\pgfpathclose
\pgfusepath{fill}
\end{pgfscope}
\begin{pgfscope}
\pgfsetdash{}{0cm}
\pgfsetlinewidth{0.818mm}
\pgfsetroundcap
\pgfsetroundjoin
\pgfsetmiterlimit{7.0}
\definecolor{eps2pgf_color}{gray}{0}\pgfsetstrokecolor{eps2pgf_color}\pgfsetfillcolor{eps2pgf_color}
\pgfpathmoveto{\pgfqpoint{0.117cm}{1.476cm}}
\pgfpathlineto{\pgfqpoint{0.682cm}{0.726cm}}
\pgfpathlineto{\pgfqpoint{1.246cm}{1.476cm}}
\pgfusepath{stroke}
\end{pgfscope}
\definecolor{eps2pgf_color}{gray}{0}\pgfsetstrokecolor{eps2pgf_color}\pgfsetfillcolor{eps2pgf_color}
\pgfpathmoveto{\pgfqpoint{0.273cm}{1.451cm}}
\pgfpathcurveto{\pgfqpoint{0.273cm}{1.487cm}}{\pgfqpoint{0.259cm}{1.522cm}}{\pgfqpoint{0.233cm}{1.547cm}}
\pgfpathcurveto{\pgfqpoint{0.207cm}{1.573cm}}{\pgfqpoint{0.173cm}{1.588cm}}{\pgfqpoint{0.137cm}{1.588cm}}
\pgfpathcurveto{\pgfqpoint{0.1cm}{1.588cm}}{\pgfqpoint{0.066cm}{1.573cm}}{\pgfqpoint{0.04cm}{1.547cm}}
\pgfpathcurveto{\pgfqpoint{0.014cm}{1.522cm}}{\pgfqpoint{0cm}{1.487cm}}{\pgfqpoint{0cm}{1.451cm}}
\pgfpathcurveto{\pgfqpoint{0cm}{1.414cm}}{\pgfqpoint{0.014cm}{1.379cm}}{\pgfqpoint{0.04cm}{1.354cm}}
\pgfpathcurveto{\pgfqpoint{0.066cm}{1.328cm}}{\pgfqpoint{0.1cm}{1.314cm}}{\pgfqpoint{0.137cm}{1.314cm}}
\pgfpathcurveto{\pgfqpoint{0.173cm}{1.314cm}}{\pgfqpoint{0.207cm}{1.328cm}}{\pgfqpoint{0.233cm}{1.354cm}}
\pgfpathcurveto{\pgfqpoint{0.259cm}{1.379cm}}{\pgfqpoint{0.273cm}{1.414cm}}{\pgfqpoint{0.273cm}{1.451cm}}
\pgfusepath{fill}
\pgfpathmoveto{\pgfqpoint{1.345cm}{1.426cm}}
\pgfpathcurveto{\pgfqpoint{1.345cm}{1.463cm}}{\pgfqpoint{1.331cm}{1.497cm}}{\pgfqpoint{1.305cm}{1.523cm}}
\pgfpathcurveto{\pgfqpoint{1.28cm}{1.549cm}}{\pgfqpoint{1.245cm}{1.563cm}}{\pgfqpoint{1.209cm}{1.563cm}}
\pgfpathcurveto{\pgfqpoint{1.172cm}{1.563cm}}{\pgfqpoint{1.138cm}{1.549cm}}{\pgfqpoint{1.112cm}{1.523cm}}
\pgfpathcurveto{\pgfqpoint{1.087cm}{1.497cm}}{\pgfqpoint{1.072cm}{1.463cm}}{\pgfqpoint{1.072cm}{1.426cm}}
\pgfpathcurveto{\pgfqpoint{1.072cm}{1.39cm}}{\pgfqpoint{1.087cm}{1.355cm}}{\pgfqpoint{1.112cm}{1.329cm}}
\pgfpathcurveto{\pgfqpoint{1.138cm}{1.304cm}}{\pgfqpoint{1.172cm}{1.289cm}}{\pgfqpoint{1.209cm}{1.289cm}}
\pgfpathcurveto{\pgfqpoint{1.245cm}{1.289cm}}{\pgfqpoint{1.28cm}{1.304cm}}{\pgfqpoint{1.305cm}{1.329cm}}
\pgfpathcurveto{\pgfqpoint{1.331cm}{1.355cm}}{\pgfqpoint{1.345cm}{1.39cm}}{\pgfqpoint{1.345cm}{1.426cm}}
\pgfusepath{fill}
\begin{pgfscope}
\pgfsetdash{}{0cm}
\pgfsetlinewidth{0.818mm}
\pgfsetroundcap
\pgfsetmiterlimit{4.0}
\pgfpathmoveto{\pgfqpoint{0.682cm}{0.726cm}}
\pgfpathlineto{\pgfqpoint{0.682cm}{0.097cm}}
\pgfusepath{stroke}
\end{pgfscope}
\end{pgfscope}
\end{pgfscope}
\end{pgfscope}
\end{tikzpicture}}}) + 3 b \zeta + 3 \llbracket
     X^2 \rrbracket \circ \eta \]
  \[ =  - 9 \zeta \left( \llbracket X^2
     \rrbracket \circ X^{\!\resizebox{0.6em}{!}{
\begin{tikzpicture}
\pgfpathmoveto{\pgfqpoint{0cm}{0cm}}
\pgfpathlineto{\pgfqpoint{1.376cm}{0cm}}
\pgfpathlineto{\pgfqpoint{1.376cm}{1.588cm}}
\pgfpathlineto{\pgfqpoint{0cm}{1.588cm}}
\pgfpathclose
\pgfusepath{clip}
\begin{pgfscope}
\begin{pgfscope}
\pgfpathmoveto{\pgfqpoint{0cm}{0cm}}
\pgfpathlineto{\pgfqpoint{1.376cm}{0cm}}
\pgfpathlineto{\pgfqpoint{1.376cm}{1.588cm}}
\pgfpathlineto{\pgfqpoint{0cm}{1.588cm}}
\pgfpathclose
\pgfusepath{clip}
\begin{pgfscope}
\begin{pgfscope}
\definecolor{eps2pgf_color}{gray}{0.976471}\pgfsetstrokecolor{eps2pgf_color}\pgfsetfillcolor{eps2pgf_color}
\pgfpathmoveto{\pgfqpoint{0cm}{0cm}}
\pgfpathlineto{\pgfqpoint{1.376cm}{0cm}}
\pgfpathlineto{\pgfqpoint{1.376cm}{1.588cm}}
\pgfpathlineto{\pgfqpoint{0cm}{1.588cm}}
\pgfpathclose
\pgfusepath{fill}
\end{pgfscope}
\begin{pgfscope}
\pgfsetdash{}{0cm}
\pgfsetlinewidth{0.818mm}
\pgfsetroundcap
\pgfsetroundjoin
\pgfsetmiterlimit{7.0}
\definecolor{eps2pgf_color}{gray}{0}\pgfsetstrokecolor{eps2pgf_color}\pgfsetfillcolor{eps2pgf_color}
\pgfpathmoveto{\pgfqpoint{0.117cm}{1.476cm}}
\pgfpathlineto{\pgfqpoint{0.682cm}{0.726cm}}
\pgfpathlineto{\pgfqpoint{1.246cm}{1.476cm}}
\pgfusepath{stroke}
\end{pgfscope}
\definecolor{eps2pgf_color}{gray}{0}\pgfsetstrokecolor{eps2pgf_color}\pgfsetfillcolor{eps2pgf_color}
\pgfpathmoveto{\pgfqpoint{0.273cm}{1.451cm}}
\pgfpathcurveto{\pgfqpoint{0.273cm}{1.487cm}}{\pgfqpoint{0.259cm}{1.522cm}}{\pgfqpoint{0.233cm}{1.547cm}}
\pgfpathcurveto{\pgfqpoint{0.207cm}{1.573cm}}{\pgfqpoint{0.173cm}{1.588cm}}{\pgfqpoint{0.137cm}{1.588cm}}
\pgfpathcurveto{\pgfqpoint{0.1cm}{1.588cm}}{\pgfqpoint{0.066cm}{1.573cm}}{\pgfqpoint{0.04cm}{1.547cm}}
\pgfpathcurveto{\pgfqpoint{0.014cm}{1.522cm}}{\pgfqpoint{0cm}{1.487cm}}{\pgfqpoint{0cm}{1.451cm}}
\pgfpathcurveto{\pgfqpoint{0cm}{1.414cm}}{\pgfqpoint{0.014cm}{1.379cm}}{\pgfqpoint{0.04cm}{1.354cm}}
\pgfpathcurveto{\pgfqpoint{0.066cm}{1.328cm}}{\pgfqpoint{0.1cm}{1.314cm}}{\pgfqpoint{0.137cm}{1.314cm}}
\pgfpathcurveto{\pgfqpoint{0.173cm}{1.314cm}}{\pgfqpoint{0.207cm}{1.328cm}}{\pgfqpoint{0.233cm}{1.354cm}}
\pgfpathcurveto{\pgfqpoint{0.259cm}{1.379cm}}{\pgfqpoint{0.273cm}{1.414cm}}{\pgfqpoint{0.273cm}{1.451cm}}
\pgfusepath{fill}
\pgfpathmoveto{\pgfqpoint{1.345cm}{1.426cm}}
\pgfpathcurveto{\pgfqpoint{1.345cm}{1.463cm}}{\pgfqpoint{1.331cm}{1.497cm}}{\pgfqpoint{1.305cm}{1.523cm}}
\pgfpathcurveto{\pgfqpoint{1.28cm}{1.549cm}}{\pgfqpoint{1.245cm}{1.563cm}}{\pgfqpoint{1.209cm}{1.563cm}}
\pgfpathcurveto{\pgfqpoint{1.172cm}{1.563cm}}{\pgfqpoint{1.138cm}{1.549cm}}{\pgfqpoint{1.112cm}{1.523cm}}
\pgfpathcurveto{\pgfqpoint{1.087cm}{1.497cm}}{\pgfqpoint{1.072cm}{1.463cm}}{\pgfqpoint{1.072cm}{1.426cm}}
\pgfpathcurveto{\pgfqpoint{1.072cm}{1.39cm}}{\pgfqpoint{1.087cm}{1.355cm}}{\pgfqpoint{1.112cm}{1.329cm}}
\pgfpathcurveto{\pgfqpoint{1.138cm}{1.304cm}}{\pgfqpoint{1.172cm}{1.289cm}}{\pgfqpoint{1.209cm}{1.289cm}}
\pgfpathcurveto{\pgfqpoint{1.245cm}{1.289cm}}{\pgfqpoint{1.28cm}{1.304cm}}{\pgfqpoint{1.305cm}{1.329cm}}
\pgfpathcurveto{\pgfqpoint{1.331cm}{1.355cm}}{\pgfqpoint{1.345cm}{1.39cm}}{\pgfqpoint{1.345cm}{1.426cm}}
\pgfusepath{fill}
\begin{pgfscope}
\pgfsetdash{}{0cm}
\pgfsetlinewidth{0.818mm}
\pgfsetroundcap
\pgfsetmiterlimit{4.0}
\pgfpathmoveto{\pgfqpoint{0.682cm}{0.726cm}}
\pgfpathlineto{\pgfqpoint{0.682cm}{0.097cm}}
\pgfusepath{stroke}
\end{pgfscope}
\end{pgfscope}
\end{pgfscope}
\end{pgfscope}
\end{tikzpicture}}} - \frac{b}{3} \right) - 9 \tmop{com} (\zeta,
     X^{\!\resizebox{0.6em}{!}{
\begin{tikzpicture}
\pgfpathmoveto{\pgfqpoint{0cm}{0cm}}
\pgfpathlineto{\pgfqpoint{1.376cm}{0cm}}
\pgfpathlineto{\pgfqpoint{1.376cm}{1.588cm}}
\pgfpathlineto{\pgfqpoint{0cm}{1.588cm}}
\pgfpathclose
\pgfusepath{clip}
\begin{pgfscope}
\begin{pgfscope}
\pgfpathmoveto{\pgfqpoint{0cm}{0cm}}
\pgfpathlineto{\pgfqpoint{1.376cm}{0cm}}
\pgfpathlineto{\pgfqpoint{1.376cm}{1.588cm}}
\pgfpathlineto{\pgfqpoint{0cm}{1.588cm}}
\pgfpathclose
\pgfusepath{clip}
\begin{pgfscope}
\begin{pgfscope}
\definecolor{eps2pgf_color}{gray}{0.976471}\pgfsetstrokecolor{eps2pgf_color}\pgfsetfillcolor{eps2pgf_color}
\pgfpathmoveto{\pgfqpoint{0cm}{0cm}}
\pgfpathlineto{\pgfqpoint{1.376cm}{0cm}}
\pgfpathlineto{\pgfqpoint{1.376cm}{1.588cm}}
\pgfpathlineto{\pgfqpoint{0cm}{1.588cm}}
\pgfpathclose
\pgfusepath{fill}
\end{pgfscope}
\begin{pgfscope}
\pgfsetdash{}{0cm}
\pgfsetlinewidth{0.818mm}
\pgfsetroundcap
\pgfsetroundjoin
\pgfsetmiterlimit{7.0}
\definecolor{eps2pgf_color}{gray}{0}\pgfsetstrokecolor{eps2pgf_color}\pgfsetfillcolor{eps2pgf_color}
\pgfpathmoveto{\pgfqpoint{0.117cm}{1.476cm}}
\pgfpathlineto{\pgfqpoint{0.682cm}{0.726cm}}
\pgfpathlineto{\pgfqpoint{1.246cm}{1.476cm}}
\pgfusepath{stroke}
\end{pgfscope}
\definecolor{eps2pgf_color}{gray}{0}\pgfsetstrokecolor{eps2pgf_color}\pgfsetfillcolor{eps2pgf_color}
\pgfpathmoveto{\pgfqpoint{0.273cm}{1.451cm}}
\pgfpathcurveto{\pgfqpoint{0.273cm}{1.487cm}}{\pgfqpoint{0.259cm}{1.522cm}}{\pgfqpoint{0.233cm}{1.547cm}}
\pgfpathcurveto{\pgfqpoint{0.207cm}{1.573cm}}{\pgfqpoint{0.173cm}{1.588cm}}{\pgfqpoint{0.137cm}{1.588cm}}
\pgfpathcurveto{\pgfqpoint{0.1cm}{1.588cm}}{\pgfqpoint{0.066cm}{1.573cm}}{\pgfqpoint{0.04cm}{1.547cm}}
\pgfpathcurveto{\pgfqpoint{0.014cm}{1.522cm}}{\pgfqpoint{0cm}{1.487cm}}{\pgfqpoint{0cm}{1.451cm}}
\pgfpathcurveto{\pgfqpoint{0cm}{1.414cm}}{\pgfqpoint{0.014cm}{1.379cm}}{\pgfqpoint{0.04cm}{1.354cm}}
\pgfpathcurveto{\pgfqpoint{0.066cm}{1.328cm}}{\pgfqpoint{0.1cm}{1.314cm}}{\pgfqpoint{0.137cm}{1.314cm}}
\pgfpathcurveto{\pgfqpoint{0.173cm}{1.314cm}}{\pgfqpoint{0.207cm}{1.328cm}}{\pgfqpoint{0.233cm}{1.354cm}}
\pgfpathcurveto{\pgfqpoint{0.259cm}{1.379cm}}{\pgfqpoint{0.273cm}{1.414cm}}{\pgfqpoint{0.273cm}{1.451cm}}
\pgfusepath{fill}
\pgfpathmoveto{\pgfqpoint{1.345cm}{1.426cm}}
\pgfpathcurveto{\pgfqpoint{1.345cm}{1.463cm}}{\pgfqpoint{1.331cm}{1.497cm}}{\pgfqpoint{1.305cm}{1.523cm}}
\pgfpathcurveto{\pgfqpoint{1.28cm}{1.549cm}}{\pgfqpoint{1.245cm}{1.563cm}}{\pgfqpoint{1.209cm}{1.563cm}}
\pgfpathcurveto{\pgfqpoint{1.172cm}{1.563cm}}{\pgfqpoint{1.138cm}{1.549cm}}{\pgfqpoint{1.112cm}{1.523cm}}
\pgfpathcurveto{\pgfqpoint{1.087cm}{1.497cm}}{\pgfqpoint{1.072cm}{1.463cm}}{\pgfqpoint{1.072cm}{1.426cm}}
\pgfpathcurveto{\pgfqpoint{1.072cm}{1.39cm}}{\pgfqpoint{1.087cm}{1.355cm}}{\pgfqpoint{1.112cm}{1.329cm}}
\pgfpathcurveto{\pgfqpoint{1.138cm}{1.304cm}}{\pgfqpoint{1.172cm}{1.289cm}}{\pgfqpoint{1.209cm}{1.289cm}}
\pgfpathcurveto{\pgfqpoint{1.245cm}{1.289cm}}{\pgfqpoint{1.28cm}{1.304cm}}{\pgfqpoint{1.305cm}{1.329cm}}
\pgfpathcurveto{\pgfqpoint{1.331cm}{1.355cm}}{\pgfqpoint{1.345cm}{1.39cm}}{\pgfqpoint{1.345cm}{1.426cm}}
\pgfusepath{fill}
\begin{pgfscope}
\pgfsetdash{}{0cm}
\pgfsetlinewidth{0.818mm}
\pgfsetroundcap
\pgfsetmiterlimit{4.0}
\pgfpathmoveto{\pgfqpoint{0.682cm}{0.726cm}}
\pgfpathlineto{\pgfqpoint{0.682cm}{0.097cm}}
\pgfusepath{stroke}
\end{pgfscope}
\end{pgfscope}
\end{pgfscope}
\end{pgfscope}
\end{tikzpicture}}}, \llbracket X^2 \rrbracket) + 3 \llbracket X^2 \rrbracket
     \circ \eta, \]
  where $X^{\!\resizebox{!}{.8em}{
\begin{tikzpicture}
\pgfpathmoveto{\pgfqpoint{0cm}{-0.035cm}}
\pgfpathlineto{\pgfqpoint{1.976cm}{-0.035cm}}
\pgfpathlineto{\pgfqpoint{1.976cm}{1.94cm}}
\pgfpathlineto{\pgfqpoint{0cm}{1.94cm}}
\pgfpathclose
\pgfusepath{clip}
\begin{pgfscope}
\begin{pgfscope}
\pgfpathmoveto{\pgfqpoint{0cm}{-0.035cm}}
\pgfpathlineto{\pgfqpoint{1.976cm}{-0.035cm}}
\pgfpathlineto{\pgfqpoint{1.976cm}{1.94cm}}
\pgfpathlineto{\pgfqpoint{0cm}{1.94cm}}
\pgfpathclose
\pgfusepath{clip}
\begin{pgfscope}
\begin{pgfscope}
\pgfsetdash{}{0cm}
\pgfsetlinewidth{0.818mm}
\pgfsetroundcap
\pgfsetroundjoin
\pgfsetmiterlimit{7.0}
\definecolor{eps2pgf_color}{gray}{0}\pgfsetstrokecolor{eps2pgf_color}\pgfsetfillcolor{eps2pgf_color}
\pgfpathmoveto{\pgfqpoint{0.117cm}{1.815cm}}
\pgfpathlineto{\pgfqpoint{0.682cm}{1.065cm}}
\pgfpathlineto{\pgfqpoint{1.246cm}{1.815cm}}
\pgfusepath{stroke}
\end{pgfscope}
\definecolor{eps2pgf_color}{gray}{0}\pgfsetstrokecolor{eps2pgf_color}\pgfsetfillcolor{eps2pgf_color}
\pgfpathmoveto{\pgfqpoint{0.273cm}{1.789cm}}
\pgfpathcurveto{\pgfqpoint{0.273cm}{1.825cm}}{\pgfqpoint{0.259cm}{1.86cm}}{\pgfqpoint{0.233cm}{1.886cm}}
\pgfpathcurveto{\pgfqpoint{0.207cm}{1.912cm}}{\pgfqpoint{0.173cm}{1.926cm}}{\pgfqpoint{0.137cm}{1.926cm}}
\pgfpathcurveto{\pgfqpoint{0.1cm}{1.926cm}}{\pgfqpoint{0.066cm}{1.912cm}}{\pgfqpoint{0.04cm}{1.886cm}}
\pgfpathcurveto{\pgfqpoint{0.014cm}{1.86cm}}{\pgfqpoint{0cm}{1.825cm}}{\pgfqpoint{0cm}{1.789cm}}
\pgfpathcurveto{\pgfqpoint{0cm}{1.753cm}}{\pgfqpoint{0.014cm}{1.718cm}}{\pgfqpoint{0.04cm}{1.692cm}}
\pgfpathcurveto{\pgfqpoint{0.066cm}{1.667cm}}{\pgfqpoint{0.1cm}{1.652cm}}{\pgfqpoint{0.137cm}{1.652cm}}
\pgfpathcurveto{\pgfqpoint{0.173cm}{1.652cm}}{\pgfqpoint{0.207cm}{1.667cm}}{\pgfqpoint{0.233cm}{1.692cm}}
\pgfpathcurveto{\pgfqpoint{0.259cm}{1.718cm}}{\pgfqpoint{0.273cm}{1.753cm}}{\pgfqpoint{0.273cm}{1.789cm}}
\pgfusepath{fill}
\pgfpathmoveto{\pgfqpoint{1.345cm}{1.765cm}}
\pgfpathcurveto{\pgfqpoint{1.345cm}{1.801cm}}{\pgfqpoint{1.331cm}{1.836cm}}{\pgfqpoint{1.305cm}{1.862cm}}
\pgfpathcurveto{\pgfqpoint{1.28cm}{1.887cm}}{\pgfqpoint{1.245cm}{1.902cm}}{\pgfqpoint{1.209cm}{1.902cm}}
\pgfpathcurveto{\pgfqpoint{1.172cm}{1.902cm}}{\pgfqpoint{1.138cm}{1.887cm}}{\pgfqpoint{1.112cm}{1.862cm}}
\pgfpathcurveto{\pgfqpoint{1.087cm}{1.836cm}}{\pgfqpoint{1.072cm}{1.801cm}}{\pgfqpoint{1.072cm}{1.765cm}}
\pgfpathcurveto{\pgfqpoint{1.072cm}{1.728cm}}{\pgfqpoint{1.087cm}{1.694cm}}{\pgfqpoint{1.112cm}{1.668cm}}
\pgfpathcurveto{\pgfqpoint{1.138cm}{1.642cm}}{\pgfqpoint{1.172cm}{1.628cm}}{\pgfqpoint{1.209cm}{1.628cm}}
\pgfpathcurveto{\pgfqpoint{1.245cm}{1.628cm}}{\pgfqpoint{1.28cm}{1.642cm}}{\pgfqpoint{1.305cm}{1.668cm}}
\pgfpathcurveto{\pgfqpoint{1.331cm}{1.694cm}}{\pgfqpoint{1.345cm}{1.728cm}}{\pgfqpoint{1.345cm}{1.765cm}}
\pgfusepath{fill}
\begin{pgfscope}
\pgfsetdash{}{0cm}
\pgfsetlinewidth{0.818mm}
\pgfsetroundcap
\pgfsetroundjoin
\pgfsetmiterlimit{7.0}
\pgfpathmoveto{\pgfqpoint{0.682cm}{1.065cm}}
\pgfpathlineto{\pgfqpoint{1.246cm}{0.315cm}}
\pgfpathlineto{\pgfqpoint{1.811cm}{1.065cm}}
\pgfusepath{stroke}
\end{pgfscope}
\pgfpathmoveto{\pgfqpoint{1.948cm}{1.065cm}}
\pgfpathcurveto{\pgfqpoint{1.948cm}{1.101cm}}{\pgfqpoint{1.933cm}{1.136cm}}{\pgfqpoint{1.907cm}{1.162cm}}
\pgfpathcurveto{\pgfqpoint{1.882cm}{1.187cm}}{\pgfqpoint{1.847cm}{1.202cm}}{\pgfqpoint{1.811cm}{1.202cm}}
\pgfpathcurveto{\pgfqpoint{1.775cm}{1.202cm}}{\pgfqpoint{1.74cm}{1.187cm}}{\pgfqpoint{1.714cm}{1.162cm}}
\pgfpathcurveto{\pgfqpoint{1.689cm}{1.136cm}}{\pgfqpoint{1.674cm}{1.101cm}}{\pgfqpoint{1.674cm}{1.065cm}}
\pgfpathcurveto{\pgfqpoint{1.674cm}{1.029cm}}{\pgfqpoint{1.689cm}{0.994cm}}{\pgfqpoint{1.714cm}{0.968cm}}
\pgfpathcurveto{\pgfqpoint{1.74cm}{0.942cm}}{\pgfqpoint{1.775cm}{0.928cm}}{\pgfqpoint{1.811cm}{0.928cm}}
\pgfpathcurveto{\pgfqpoint{1.847cm}{0.928cm}}{\pgfqpoint{1.882cm}{0.942cm}}{\pgfqpoint{1.907cm}{0.968cm}}
\pgfpathcurveto{\pgfqpoint{1.933cm}{0.994cm}}{\pgfqpoint{1.948cm}{1.029cm}}{\pgfqpoint{1.948cm}{1.065cm}}
\pgfusepath{fill}
\begin{pgfscope}
\pgfsetdash{}{0cm}
\pgfsetlinewidth{0.818mm}
\pgfsetmiterlimit{7.0}
\pgfpathmoveto{\pgfqpoint{1.246cm}{0.315cm}}
\pgfpathlineto{\pgfqpoint{1.244cm}{1.061cm}}
\pgfusepath{stroke}
\end{pgfscope}
\pgfpathmoveto{\pgfqpoint{1.38cm}{1.065cm}}
\pgfpathcurveto{\pgfqpoint{1.38cm}{1.101cm}}{\pgfqpoint{1.366cm}{1.136cm}}{\pgfqpoint{1.34cm}{1.162cm}}
\pgfpathcurveto{\pgfqpoint{1.315cm}{1.187cm}}{\pgfqpoint{1.28cm}{1.202cm}}{\pgfqpoint{1.244cm}{1.202cm}}
\pgfpathcurveto{\pgfqpoint{1.207cm}{1.202cm}}{\pgfqpoint{1.173cm}{1.187cm}}{\pgfqpoint{1.147cm}{1.162cm}}
\pgfpathcurveto{\pgfqpoint{1.121cm}{1.136cm}}{\pgfqpoint{1.107cm}{1.101cm}}{\pgfqpoint{1.107cm}{1.065cm}}
\pgfpathcurveto{\pgfqpoint{1.107cm}{1.029cm}}{\pgfqpoint{1.121cm}{0.994cm}}{\pgfqpoint{1.147cm}{0.968cm}}
\pgfpathcurveto{\pgfqpoint{1.173cm}{0.942cm}}{\pgfqpoint{1.207cm}{0.928cm}}{\pgfqpoint{1.244cm}{0.928cm}}
\pgfpathcurveto{\pgfqpoint{1.28cm}{0.928cm}}{\pgfqpoint{1.315cm}{0.942cm}}{\pgfqpoint{1.34cm}{0.968cm}}
\pgfpathcurveto{\pgfqpoint{1.366cm}{0.994cm}}{\pgfqpoint{1.38cm}{1.029cm}}{\pgfqpoint{1.38cm}{1.065cm}}
\pgfusepath{fill}
\begin{pgfscope}
\pgfsetdash{}{0cm}
\pgfsetlinewidth{0.818mm}
\pgfsetmiterlimit{4.0}
\pgfpathmoveto{\pgfqpoint{1.383cm}{0.178cm}}
\pgfpathcurveto{\pgfqpoint{1.383cm}{0.214cm}}{\pgfqpoint{1.369cm}{0.249cm}}{\pgfqpoint{1.343cm}{0.275cm}}
\pgfpathcurveto{\pgfqpoint{1.317cm}{0.3cm}}{\pgfqpoint{1.283cm}{0.315cm}}{\pgfqpoint{1.246cm}{0.315cm}}
\pgfpathcurveto{\pgfqpoint{1.21cm}{0.315cm}}{\pgfqpoint{1.175cm}{0.3cm}}{\pgfqpoint{1.15cm}{0.275cm}}
\pgfpathcurveto{\pgfqpoint{1.124cm}{0.249cm}}{\pgfqpoint{1.11cm}{0.214cm}}{\pgfqpoint{1.11cm}{0.178cm}}
\pgfpathcurveto{\pgfqpoint{1.11cm}{0.141cm}}{\pgfqpoint{1.124cm}{0.107cm}}{\pgfqpoint{1.15cm}{0.081cm}}
\pgfpathcurveto{\pgfqpoint{1.175cm}{0.055cm}}{\pgfqpoint{1.21cm}{0.041cm}}{\pgfqpoint{1.246cm}{0.041cm}}
\pgfpathcurveto{\pgfqpoint{1.283cm}{0.041cm}}{\pgfqpoint{1.317cm}{0.055cm}}{\pgfqpoint{1.343cm}{0.081cm}}
\pgfpathcurveto{\pgfqpoint{1.369cm}{0.107cm}}{\pgfqpoint{1.383cm}{0.141cm}}{\pgfqpoint{1.383cm}{0.178cm}}
\pgfusepath{stroke}
\end{pgfscope}
\end{pgfscope}
\end{pgfscope}
\end{pgfscope}
\end{tikzpicture}}}=\llbracket X^2 \rrbracket \circ X^{\!\resizebox{0.6em}{!}{
\begin{tikzpicture}
\pgfpathmoveto{\pgfqpoint{0cm}{0cm}}
\pgfpathlineto{\pgfqpoint{1.376cm}{0cm}}
\pgfpathlineto{\pgfqpoint{1.376cm}{1.588cm}}
\pgfpathlineto{\pgfqpoint{0cm}{1.588cm}}
\pgfpathclose
\pgfusepath{clip}
\begin{pgfscope}
\begin{pgfscope}
\pgfpathmoveto{\pgfqpoint{0cm}{0cm}}
\pgfpathlineto{\pgfqpoint{1.376cm}{0cm}}
\pgfpathlineto{\pgfqpoint{1.376cm}{1.588cm}}
\pgfpathlineto{\pgfqpoint{0cm}{1.588cm}}
\pgfpathclose
\pgfusepath{clip}
\begin{pgfscope}
\begin{pgfscope}
\definecolor{eps2pgf_color}{gray}{0.976471}\pgfsetstrokecolor{eps2pgf_color}\pgfsetfillcolor{eps2pgf_color}
\pgfpathmoveto{\pgfqpoint{0cm}{0cm}}
\pgfpathlineto{\pgfqpoint{1.376cm}{0cm}}
\pgfpathlineto{\pgfqpoint{1.376cm}{1.588cm}}
\pgfpathlineto{\pgfqpoint{0cm}{1.588cm}}
\pgfpathclose
\pgfusepath{fill}
\end{pgfscope}
\begin{pgfscope}
\pgfsetdash{}{0cm}
\pgfsetlinewidth{0.818mm}
\pgfsetroundcap
\pgfsetroundjoin
\pgfsetmiterlimit{7.0}
\definecolor{eps2pgf_color}{gray}{0}\pgfsetstrokecolor{eps2pgf_color}\pgfsetfillcolor{eps2pgf_color}
\pgfpathmoveto{\pgfqpoint{0.117cm}{1.476cm}}
\pgfpathlineto{\pgfqpoint{0.682cm}{0.726cm}}
\pgfpathlineto{\pgfqpoint{1.246cm}{1.476cm}}
\pgfusepath{stroke}
\end{pgfscope}
\definecolor{eps2pgf_color}{gray}{0}\pgfsetstrokecolor{eps2pgf_color}\pgfsetfillcolor{eps2pgf_color}
\pgfpathmoveto{\pgfqpoint{0.273cm}{1.451cm}}
\pgfpathcurveto{\pgfqpoint{0.273cm}{1.487cm}}{\pgfqpoint{0.259cm}{1.522cm}}{\pgfqpoint{0.233cm}{1.547cm}}
\pgfpathcurveto{\pgfqpoint{0.207cm}{1.573cm}}{\pgfqpoint{0.173cm}{1.588cm}}{\pgfqpoint{0.137cm}{1.588cm}}
\pgfpathcurveto{\pgfqpoint{0.1cm}{1.588cm}}{\pgfqpoint{0.066cm}{1.573cm}}{\pgfqpoint{0.04cm}{1.547cm}}
\pgfpathcurveto{\pgfqpoint{0.014cm}{1.522cm}}{\pgfqpoint{0cm}{1.487cm}}{\pgfqpoint{0cm}{1.451cm}}
\pgfpathcurveto{\pgfqpoint{0cm}{1.414cm}}{\pgfqpoint{0.014cm}{1.379cm}}{\pgfqpoint{0.04cm}{1.354cm}}
\pgfpathcurveto{\pgfqpoint{0.066cm}{1.328cm}}{\pgfqpoint{0.1cm}{1.314cm}}{\pgfqpoint{0.137cm}{1.314cm}}
\pgfpathcurveto{\pgfqpoint{0.173cm}{1.314cm}}{\pgfqpoint{0.207cm}{1.328cm}}{\pgfqpoint{0.233cm}{1.354cm}}
\pgfpathcurveto{\pgfqpoint{0.259cm}{1.379cm}}{\pgfqpoint{0.273cm}{1.414cm}}{\pgfqpoint{0.273cm}{1.451cm}}
\pgfusepath{fill}
\pgfpathmoveto{\pgfqpoint{1.345cm}{1.426cm}}
\pgfpathcurveto{\pgfqpoint{1.345cm}{1.463cm}}{\pgfqpoint{1.331cm}{1.497cm}}{\pgfqpoint{1.305cm}{1.523cm}}
\pgfpathcurveto{\pgfqpoint{1.28cm}{1.549cm}}{\pgfqpoint{1.245cm}{1.563cm}}{\pgfqpoint{1.209cm}{1.563cm}}
\pgfpathcurveto{\pgfqpoint{1.172cm}{1.563cm}}{\pgfqpoint{1.138cm}{1.549cm}}{\pgfqpoint{1.112cm}{1.523cm}}
\pgfpathcurveto{\pgfqpoint{1.087cm}{1.497cm}}{\pgfqpoint{1.072cm}{1.463cm}}{\pgfqpoint{1.072cm}{1.426cm}}
\pgfpathcurveto{\pgfqpoint{1.072cm}{1.39cm}}{\pgfqpoint{1.087cm}{1.355cm}}{\pgfqpoint{1.112cm}{1.329cm}}
\pgfpathcurveto{\pgfqpoint{1.138cm}{1.304cm}}{\pgfqpoint{1.172cm}{1.289cm}}{\pgfqpoint{1.209cm}{1.289cm}}
\pgfpathcurveto{\pgfqpoint{1.245cm}{1.289cm}}{\pgfqpoint{1.28cm}{1.304cm}}{\pgfqpoint{1.305cm}{1.329cm}}
\pgfpathcurveto{\pgfqpoint{1.331cm}{1.355cm}}{\pgfqpoint{1.345cm}{1.39cm}}{\pgfqpoint{1.345cm}{1.426cm}}
\pgfusepath{fill}
\begin{pgfscope}
\pgfsetdash{}{0cm}
\pgfsetlinewidth{0.818mm}
\pgfsetroundcap
\pgfsetmiterlimit{4.0}
\pgfpathmoveto{\pgfqpoint{0.682cm}{0.726cm}}
\pgfpathlineto{\pgfqpoint{0.682cm}{0.097cm}}
\pgfusepath{stroke}
\end{pgfscope}
\end{pgfscope}
\end{pgfscope}
\end{pgfscope}
\end{tikzpicture}}} - \frac{b}{3}$ is the
  renormalized resonant product in $\CC^{- \kappa} (\rho^{\sigma})$.
  
  Now, we introduce a time dependent weight $\pi (t, x) = \exp (- t \rho^{- 2
  b} (x))$ where $\rho (x) = \langle x \rangle^{- 1}$ is a polynomial weight
  and $b \in (0, 1 / 2)$. Let $\beta, \gamma \in (0, 1)$ to be chosen below.
  First, we use the fact that $\Delta_k \LL = \LL \Delta_k$ in order to derive
  the equation for the Littlewood--Paley block $\Delta_k \eta$. Now, we test
  this equation by $\pi^2 \Delta_k \eta$ and use $\partial_t \pi = - \pi
  \rho^{- 2 b}$ to get
  \begin{equation}
    \frac{1}{2} \partial_t \| \Delta_k \eta \|_{L^2 (\pi)}^2 + \| \nabla
    \Delta_k \eta \|_{L^2 (\pi)}^2 + \mu \| \Delta_k \eta \|_{L^2 (\pi)}^2 +
    \| \rho^{- b} \Delta_k \eta \|_{L^2 (\pi)}^2 \label{eq:eta}
  \end{equation}
  \[ = \left\langle \pi^{} \Delta_k \eta, \pi \Delta_k \left( \LL \eta \right)
     \right\rangle - 2 \langle \pi \Delta_k \eta, \frac{\nabla \pi}{\pi} \pi
     \nabla \Delta_k \eta \rangle . \]
  Since $\nabla \pi = \pi t 2 b \rho^{- 2 b - 1} \nabla \rho$ and $| \nabla
  \rho / \rho^2 | \lesssim 1$, we obtain $\left| \frac{\nabla \pi}{\pi}
  \right| \lesssim t \rho^{1 - 2 b} \lesssim_T 1$ as a consequence of $b \in
  (0, 1 / 2)$ and $t \in [0, T]$. Hence by Young's inequality
  \[ | \langle \pi \Delta_k \eta, \frac{\nabla \pi}{\pi} \pi \nabla \Delta_k
     \eta \rangle | \leqslant C_{T, \delta} \|  \Delta_k \eta
     \|_{L^2 (\pi)}^2 + \delta \| \nabla \Delta_k \eta \|_{L^2 (\pi)}^2 . \]
  Now we estimate by duality, for some $\gamma \in (0, 1)$ and some $0 < a < b
  / 2$,
  \[ \left| \left\langle \pi^{} \Delta_k \eta, \pi \Delta_k \left( \LL \eta
     \right) \right\rangle \right| \lesssim 2^{- 2 \beta k} 2^{(2 \beta +
     \gamma) k} \| \Delta_k \eta \|_{L^2 (\pi \rho^{- a})} 2^{- \gamma k}
     \left\| \Delta_k \left( \LL \eta \right) \right\|_{L^2 (\pi \rho^a)} \]
  \[ \leqslant 2^{- 2 \beta k} \left( C_{\delta} 2^{2 (2 \beta + \gamma) k} \|
     \Delta_k \eta \|^2_{L^2 (\pi \rho^{- a})} + \delta 2^{- 2 \gamma k}
     \left\| \Delta_k \left( \LL \eta \right) \right\|^2_{L^2 (\pi \rho^a)}
     \right) . \]
 \rmbb{ Moreover, by a suitable choice of $a \in (0, b / 2)$ and $\beta, \gamma \in
  (0, 1)$ such that $\beta < 2 \beta + \gamma < \beta + 1$ (hence $\beta +
  \gamma < 1$), we may interpolate as follows. Let $\lambda\in (0,1)$ such that $a=(1-\lambda) b/2$ and $2\beta+\gamma=(1-\lambda)\beta+\lambda(\beta+1)$. Then H\"older's, Bernstein's and Young's  inequalities imply
  \[ 2^{2 (2 \beta + \gamma) k} \| \Delta_k \eta \|^2_{L^2 (\pi \rho^{-
     a})} \leqslant 2^{2(1-\lambda)\beta k} 2^{2\lambda(\beta+1)k} \|\rho^{-(1-\lambda)b/2} |\Delta_k \eta|^{1-\lambda} |\Delta_k \eta|^{\lambda} \|^2_{L^2 (\pi )}
     \]
\[
\leqslant 2^{2 ( 1-\lambda ) \beta k}  \| \rho^{- b/2} \Delta_k \eta \|^{2(1-\lambda)}_{ L^2 (\pi ) }
	2^{ 2 \lambda ( \beta+1 ) k}\|  \Delta_k \eta \|^{2\lambda}_{ L^2 (\pi ) }
\] 
\[
\leqslant 2^{2 ( 1-\lambda ) \beta k}  \| \rho^{- b/2} \Delta_k \eta \|^{2(1-\lambda)}_{ L^2 (\pi ) }
	2^{ 2 \lambda \beta k}\| \nabla \Delta_k \eta \|^{2\lambda}_{ L^2 (\pi ) }
\]     
\[
     \leqslant C_{\delta} 2^{2 \beta k} \| \rho^{- b / 2} \Delta_k
     \eta \|^2_{L^2 (\pi)} + \delta 2^{2 \beta k} \| \nabla \Delta_k \eta
     \|^2_{L^2 (\pi)} \]
  \[ \leqslant C_{\delta} 2^{2 \beta k} \| \Delta_k \eta \|^2_{L^2 (\pi)} +
     \delta 2^{2 \beta k} \| \rho^{- b} \Delta_k \eta \|^2_{L^2 (\pi)} +
     \delta 2^{2 \beta k} \| \nabla \Delta_k \eta \|_{L^2 (\pi)}^2 . \]}
  Similarly, we may test the equation for $\Delta_k \zeta$ by $\pi^2 \Delta_k
  \zeta$ to obtain
  \begin{equation}
    \frac{1}{2} \partial_t \| \Delta_k \zeta \|_{L^2 (\pi)}^2 + \| \nabla
    \Delta_k \zeta \|_{L^2 (\pi)}^2 + \mu \| \Delta_k \zeta \|_{L^2 (\pi)}^2 +
    \| \rho^{- b} \Delta_k \zeta \|_{L^2 (\pi)}^2 \label{eq:zeta}
  \end{equation}
  \[ = \left\langle \pi^{} \Delta_k \zeta, \pi \Delta_k \left( \LL \zeta
     \right) \right\rangle - 2 \langle \pi \Delta_k \zeta, \frac{\nabla
     \pi}{\pi} \pi \nabla \Delta_k \zeta \rangle . \]
  Next, we estimate
  \[ | \langle \pi \Delta_k \zeta, \frac{\nabla \pi}{\pi} \pi \nabla \Delta_k
     \zeta \rangle | \leqslant C_{T, \delta} \| \rho^{1 - 2 b} \Delta_k \zeta
     \|_{L^2 (\pi)}^2 + \delta \| \nabla \Delta_k \zeta \|_{L^2 (\pi)}^2 \]
  \[ \leqslant C_{T, \delta} \| \Delta_k \zeta \|_{L^2 (\pi)}^2 + \delta \|
    \nabla \Delta_k \zeta \|_{L^2 (\pi)}^2, \]
  and by duality
  \[ \left| \left\langle \pi^2 \Delta_k \zeta, \Delta_k \LL \zeta
     \right\rangle \right| \lesssim 2^{2 \beta k} 2^{(1 + \kappa - 2 \beta) k}
     \| \Delta_k \zeta \|_{L^2 (\pi \rho^{- c})} 2^{- (1 + \kappa) k} \left\|
     \Delta_k \left( \LL \zeta \right) \right\|_{L^2 (\pi \rho^c)} \]
  \[ \leqslant 2^{2 \beta k} \left( C_{\delta} 2^{2 (1 + \kappa - 2 \beta) k}
     \| \Delta_k \zeta \|^2_{L^2 (\pi \rho^{- c})} + \delta 2^{- 2 (1 +
     \kappa) k} \left\| \Delta_k \left( \LL \zeta \right) \right\|^2_{L^2 (\pi
     \rho^c)} \right), \]
  where by a suitable choice of $c \in (0, b / 2)$ we again interpolate to obtain
  \[ 2^{2 (1 + \kappa - 2 \beta) k} \| \Delta_k \zeta \|^2_{L^2 (\pi \rho^{-
     c})} \leqslant C_{\delta} 2^{- 2 \beta k} \| \rho^{- b / 2} \Delta_k
     \zeta \|^2_{L^2 (\pi)} + \delta 2^{- 2 \beta k} \| \nabla \Delta_k \zeta
     \|^2_{L^2 (\pi)} \]
  \[ \leqslant C_{\delta} 2^{- 2 \beta k} \| \Delta_k \zeta \|^2_{L^2
     (\pi)} + \delta 2^{- 2 \beta k} \| \rho^{- b} \Delta_k \zeta \|^2_{L^2
     (\pi)} + \delta 2^{- 2 \beta k} \| \nabla \Delta_k \zeta \|_{L^2
     (\pi)}^2, \]
  provided $- \beta < 1 + \kappa - 2 \beta < 1 - \beta$ hence $\kappa < \beta
  < 1 + \kappa$.
  
  As the next step, we multiply (\ref{eq:eta}) by $2^{2 \beta k}$,
  $(\ref{eq:zeta})$ by $2^{- 2 \beta k}$, integrate both inequalities over
  $(0, t)$ for some $t \in (0, T]$. In addition, we choose $\delta$
  sufficiently small  in order to absorb some of the
  terms into the left hand side. Finally, we sum the two inequalities, use the
  fact that $\zeta (0) = 0, \eta (0) = 0$, and sum over $k$ to obtain
  \[ \frac{1}{2} \| \eta (t) \|_{B^{\beta}_{2, 2} (\pi)}^2 + \frac{1}{2}\int_0^t \|
     \nabla \eta \|_{B^{\beta}_{2, 2} (\pi)}^2 \mathd s + \frac{1}{2} \| \zeta
     (t) \|_{B^{- \beta}_{2, 2} (\pi)}^2 + \frac{1}{2}\int_0^t \| \nabla \zeta \|_{B^{-
     \beta}_{2, 2} (\pi)}^2 \mathd s \]
  \begin{equation}\label{eq:s1}
    \le C_{\delta} \int_{0}^{t}\| \eta \|^2_{B^{\beta}_{2,2}
     (\pi)}\mathd s +C_{\delta} \int_{0}^{t}\| \zeta \|^2_{B^{-\beta}_{2,2}
     (\pi)}\mathd s 
     \end{equation}
     \begin{equation}+ \delta \int_0^t \left\| \LL \eta \right\|^2_{B^{- \gamma}_{2, 2}
    (\pi \rho^a)} \mathd s + \delta \int_0^t \left\| \LL \zeta \right\|^2_{B^{-
    (1 + \kappa)}_{2, 2} (\pi \rho^c)} \mathd s. \label{eq:s}
  \end{equation}
  We remark that in the above inequality we omitted the two terms containing $\rho^{-b}$ as well as the two terms containing $\mu$ on the left hand side since they are no longer needed.
  So it remains to control $\left\| \LL \eta \right\|^2_{B^{- \gamma}_{2, 2}
  (\pi \rho^a)}$ and $\left\| \LL \zeta \right\|^2_{B^{- (1 + \kappa)}_{2, 2}
  (\pi \rho^c)}$. Note that both terms are multiplied by a small constant
  $\delta > 0$, which will be needed in order to absorb them into the left hand side.
  
 We set $\beta = 2 \kappa$, $\gamma = \frac 12 +\kappa$ for some $\kappa > 0$
  sufficiently small, which is also the parameter to be used in order to
  estimate the stochastic terms according to the Table~\ref{t:reg}. Then $a=(1/2-3\kappa)b/2$ whereas $c=\kappa b/2$
hence  $c<a$, which will be used below in order to control the time derivative of $\zeta$.
  
 In view of Lemma \ref{lem:par1} and Lemma
  \ref{lem:com1}, we may estimate all the terms in $\LL \eta$ as follows
  \[ \| 3 \llbracket X^2 \rrbracket \prec \zeta (t) \|_{B^{- \gamma}_{2, 2}
     (\pi_{t} \rho^a)} \lesssim \| \zeta (t) \|_{B^{1 - \beta}_{2, 2} (\pi_{t})}, \]
  \[ \left\| 9 \zeta \left( \llbracket X^2 \rrbracket \circ X^{\!\resizebox{0.6em}{!}{
\begin{tikzpicture}
\pgfpathmoveto{\pgfqpoint{0cm}{0cm}}
\pgfpathlineto{\pgfqpoint{1.376cm}{0cm}}
\pgfpathlineto{\pgfqpoint{1.376cm}{1.588cm}}
\pgfpathlineto{\pgfqpoint{0cm}{1.588cm}}
\pgfpathclose
\pgfusepath{clip}
\begin{pgfscope}
\begin{pgfscope}
\pgfpathmoveto{\pgfqpoint{0cm}{0cm}}
\pgfpathlineto{\pgfqpoint{1.376cm}{0cm}}
\pgfpathlineto{\pgfqpoint{1.376cm}{1.588cm}}
\pgfpathlineto{\pgfqpoint{0cm}{1.588cm}}
\pgfpathclose
\pgfusepath{clip}
\begin{pgfscope}
\begin{pgfscope}
\definecolor{eps2pgf_color}{gray}{0.976471}\pgfsetstrokecolor{eps2pgf_color}\pgfsetfillcolor{eps2pgf_color}
\pgfpathmoveto{\pgfqpoint{0cm}{0cm}}
\pgfpathlineto{\pgfqpoint{1.376cm}{0cm}}
\pgfpathlineto{\pgfqpoint{1.376cm}{1.588cm}}
\pgfpathlineto{\pgfqpoint{0cm}{1.588cm}}
\pgfpathclose
\pgfusepath{fill}
\end{pgfscope}
\begin{pgfscope}
\pgfsetdash{}{0cm}
\pgfsetlinewidth{0.818mm}
\pgfsetroundcap
\pgfsetroundjoin
\pgfsetmiterlimit{7.0}
\definecolor{eps2pgf_color}{gray}{0}\pgfsetstrokecolor{eps2pgf_color}\pgfsetfillcolor{eps2pgf_color}
\pgfpathmoveto{\pgfqpoint{0.117cm}{1.476cm}}
\pgfpathlineto{\pgfqpoint{0.682cm}{0.726cm}}
\pgfpathlineto{\pgfqpoint{1.246cm}{1.476cm}}
\pgfusepath{stroke}
\end{pgfscope}
\definecolor{eps2pgf_color}{gray}{0}\pgfsetstrokecolor{eps2pgf_color}\pgfsetfillcolor{eps2pgf_color}
\pgfpathmoveto{\pgfqpoint{0.273cm}{1.451cm}}
\pgfpathcurveto{\pgfqpoint{0.273cm}{1.487cm}}{\pgfqpoint{0.259cm}{1.522cm}}{\pgfqpoint{0.233cm}{1.547cm}}
\pgfpathcurveto{\pgfqpoint{0.207cm}{1.573cm}}{\pgfqpoint{0.173cm}{1.588cm}}{\pgfqpoint{0.137cm}{1.588cm}}
\pgfpathcurveto{\pgfqpoint{0.1cm}{1.588cm}}{\pgfqpoint{0.066cm}{1.573cm}}{\pgfqpoint{0.04cm}{1.547cm}}
\pgfpathcurveto{\pgfqpoint{0.014cm}{1.522cm}}{\pgfqpoint{0cm}{1.487cm}}{\pgfqpoint{0cm}{1.451cm}}
\pgfpathcurveto{\pgfqpoint{0cm}{1.414cm}}{\pgfqpoint{0.014cm}{1.379cm}}{\pgfqpoint{0.04cm}{1.354cm}}
\pgfpathcurveto{\pgfqpoint{0.066cm}{1.328cm}}{\pgfqpoint{0.1cm}{1.314cm}}{\pgfqpoint{0.137cm}{1.314cm}}
\pgfpathcurveto{\pgfqpoint{0.173cm}{1.314cm}}{\pgfqpoint{0.207cm}{1.328cm}}{\pgfqpoint{0.233cm}{1.354cm}}
\pgfpathcurveto{\pgfqpoint{0.259cm}{1.379cm}}{\pgfqpoint{0.273cm}{1.414cm}}{\pgfqpoint{0.273cm}{1.451cm}}
\pgfusepath{fill}
\pgfpathmoveto{\pgfqpoint{1.345cm}{1.426cm}}
\pgfpathcurveto{\pgfqpoint{1.345cm}{1.463cm}}{\pgfqpoint{1.331cm}{1.497cm}}{\pgfqpoint{1.305cm}{1.523cm}}
\pgfpathcurveto{\pgfqpoint{1.28cm}{1.549cm}}{\pgfqpoint{1.245cm}{1.563cm}}{\pgfqpoint{1.209cm}{1.563cm}}
\pgfpathcurveto{\pgfqpoint{1.172cm}{1.563cm}}{\pgfqpoint{1.138cm}{1.549cm}}{\pgfqpoint{1.112cm}{1.523cm}}
\pgfpathcurveto{\pgfqpoint{1.087cm}{1.497cm}}{\pgfqpoint{1.072cm}{1.463cm}}{\pgfqpoint{1.072cm}{1.426cm}}
\pgfpathcurveto{\pgfqpoint{1.072cm}{1.39cm}}{\pgfqpoint{1.087cm}{1.355cm}}{\pgfqpoint{1.112cm}{1.329cm}}
\pgfpathcurveto{\pgfqpoint{1.138cm}{1.304cm}}{\pgfqpoint{1.172cm}{1.289cm}}{\pgfqpoint{1.209cm}{1.289cm}}
\pgfpathcurveto{\pgfqpoint{1.245cm}{1.289cm}}{\pgfqpoint{1.28cm}{1.304cm}}{\pgfqpoint{1.305cm}{1.329cm}}
\pgfpathcurveto{\pgfqpoint{1.331cm}{1.355cm}}{\pgfqpoint{1.345cm}{1.39cm}}{\pgfqpoint{1.345cm}{1.426cm}}
\pgfusepath{fill}
\begin{pgfscope}
\pgfsetdash{}{0cm}
\pgfsetlinewidth{0.818mm}
\pgfsetroundcap
\pgfsetmiterlimit{4.0}
\pgfpathmoveto{\pgfqpoint{0.682cm}{0.726cm}}
\pgfpathlineto{\pgfqpoint{0.682cm}{0.097cm}}
\pgfusepath{stroke}
\end{pgfscope}
\end{pgfscope}
\end{pgfscope}
\end{pgfscope}
\end{tikzpicture}}} -
     \frac{b}{3} \right) (t) \right\|_{B^{- \gamma}_{2, 2} (\pi _{t}\rho^a)}
     \lesssim \| \zeta (t) \|_{B^{1 - \beta}_{2, 2} (\pi_{t})}, \]
  \[ \left\| 9 \tmop{com} (\zeta, X^{\!\resizebox{0.6em}{!}{
\begin{tikzpicture}
\pgfpathmoveto{\pgfqpoint{0cm}{0cm}}
\pgfpathlineto{\pgfqpoint{1.376cm}{0cm}}
\pgfpathlineto{\pgfqpoint{1.376cm}{1.588cm}}
\pgfpathlineto{\pgfqpoint{0cm}{1.588cm}}
\pgfpathclose
\pgfusepath{clip}
\begin{pgfscope}
\begin{pgfscope}
\pgfpathmoveto{\pgfqpoint{0cm}{0cm}}
\pgfpathlineto{\pgfqpoint{1.376cm}{0cm}}
\pgfpathlineto{\pgfqpoint{1.376cm}{1.588cm}}
\pgfpathlineto{\pgfqpoint{0cm}{1.588cm}}
\pgfpathclose
\pgfusepath{clip}
\begin{pgfscope}
\begin{pgfscope}
\definecolor{eps2pgf_color}{gray}{0.976471}\pgfsetstrokecolor{eps2pgf_color}\pgfsetfillcolor{eps2pgf_color}
\pgfpathmoveto{\pgfqpoint{0cm}{0cm}}
\pgfpathlineto{\pgfqpoint{1.376cm}{0cm}}
\pgfpathlineto{\pgfqpoint{1.376cm}{1.588cm}}
\pgfpathlineto{\pgfqpoint{0cm}{1.588cm}}
\pgfpathclose
\pgfusepath{fill}
\end{pgfscope}
\begin{pgfscope}
\pgfsetdash{}{0cm}
\pgfsetlinewidth{0.818mm}
\pgfsetroundcap
\pgfsetroundjoin
\pgfsetmiterlimit{7.0}
\definecolor{eps2pgf_color}{gray}{0}\pgfsetstrokecolor{eps2pgf_color}\pgfsetfillcolor{eps2pgf_color}
\pgfpathmoveto{\pgfqpoint{0.117cm}{1.476cm}}
\pgfpathlineto{\pgfqpoint{0.682cm}{0.726cm}}
\pgfpathlineto{\pgfqpoint{1.246cm}{1.476cm}}
\pgfusepath{stroke}
\end{pgfscope}
\definecolor{eps2pgf_color}{gray}{0}\pgfsetstrokecolor{eps2pgf_color}\pgfsetfillcolor{eps2pgf_color}
\pgfpathmoveto{\pgfqpoint{0.273cm}{1.451cm}}
\pgfpathcurveto{\pgfqpoint{0.273cm}{1.487cm}}{\pgfqpoint{0.259cm}{1.522cm}}{\pgfqpoint{0.233cm}{1.547cm}}
\pgfpathcurveto{\pgfqpoint{0.207cm}{1.573cm}}{\pgfqpoint{0.173cm}{1.588cm}}{\pgfqpoint{0.137cm}{1.588cm}}
\pgfpathcurveto{\pgfqpoint{0.1cm}{1.588cm}}{\pgfqpoint{0.066cm}{1.573cm}}{\pgfqpoint{0.04cm}{1.547cm}}
\pgfpathcurveto{\pgfqpoint{0.014cm}{1.522cm}}{\pgfqpoint{0cm}{1.487cm}}{\pgfqpoint{0cm}{1.451cm}}
\pgfpathcurveto{\pgfqpoint{0cm}{1.414cm}}{\pgfqpoint{0.014cm}{1.379cm}}{\pgfqpoint{0.04cm}{1.354cm}}
\pgfpathcurveto{\pgfqpoint{0.066cm}{1.328cm}}{\pgfqpoint{0.1cm}{1.314cm}}{\pgfqpoint{0.137cm}{1.314cm}}
\pgfpathcurveto{\pgfqpoint{0.173cm}{1.314cm}}{\pgfqpoint{0.207cm}{1.328cm}}{\pgfqpoint{0.233cm}{1.354cm}}
\pgfpathcurveto{\pgfqpoint{0.259cm}{1.379cm}}{\pgfqpoint{0.273cm}{1.414cm}}{\pgfqpoint{0.273cm}{1.451cm}}
\pgfusepath{fill}
\pgfpathmoveto{\pgfqpoint{1.345cm}{1.426cm}}
\pgfpathcurveto{\pgfqpoint{1.345cm}{1.463cm}}{\pgfqpoint{1.331cm}{1.497cm}}{\pgfqpoint{1.305cm}{1.523cm}}
\pgfpathcurveto{\pgfqpoint{1.28cm}{1.549cm}}{\pgfqpoint{1.245cm}{1.563cm}}{\pgfqpoint{1.209cm}{1.563cm}}
\pgfpathcurveto{\pgfqpoint{1.172cm}{1.563cm}}{\pgfqpoint{1.138cm}{1.549cm}}{\pgfqpoint{1.112cm}{1.523cm}}
\pgfpathcurveto{\pgfqpoint{1.087cm}{1.497cm}}{\pgfqpoint{1.072cm}{1.463cm}}{\pgfqpoint{1.072cm}{1.426cm}}
\pgfpathcurveto{\pgfqpoint{1.072cm}{1.39cm}}{\pgfqpoint{1.087cm}{1.355cm}}{\pgfqpoint{1.112cm}{1.329cm}}
\pgfpathcurveto{\pgfqpoint{1.138cm}{1.304cm}}{\pgfqpoint{1.172cm}{1.289cm}}{\pgfqpoint{1.209cm}{1.289cm}}
\pgfpathcurveto{\pgfqpoint{1.245cm}{1.289cm}}{\pgfqpoint{1.28cm}{1.304cm}}{\pgfqpoint{1.305cm}{1.329cm}}
\pgfpathcurveto{\pgfqpoint{1.331cm}{1.355cm}}{\pgfqpoint{1.345cm}{1.39cm}}{\pgfqpoint{1.345cm}{1.426cm}}
\pgfusepath{fill}
\begin{pgfscope}
\pgfsetdash{}{0cm}
\pgfsetlinewidth{0.818mm}
\pgfsetroundcap
\pgfsetmiterlimit{4.0}
\pgfpathmoveto{\pgfqpoint{0.682cm}{0.726cm}}
\pgfpathlineto{\pgfqpoint{0.682cm}{0.097cm}}
\pgfusepath{stroke}
\end{pgfscope}
\end{pgfscope}
\end{pgfscope}
\end{pgfscope}
\end{tikzpicture}}}, \llbracket X^2 \rrbracket) (t)
     \right\|_{B^{- \gamma}_{2, 2} (\pi_{t} \rho^a)} \lesssim \| \zeta (t)
     \|_{B^{1 - \beta}_{2, 2} (\pi_{t})}, \]
  \[ \| 3 \llbracket X^2 \rrbracket \circ \eta (t) \|_{B^{- \gamma}_{2, 2}
     (\pi _{t}\rho^a)} \lesssim \| \eta (t) \|_{B^{1 + \beta}_{2, 2} (\pi_{t})}, \]
  \[ \left\| 3 [\LL, \zeta \prec] X^{\!\resizebox{0.6em}{!}{
\begin{tikzpicture}
\pgfpathmoveto{\pgfqpoint{0cm}{0cm}}
\pgfpathlineto{\pgfqpoint{1.376cm}{0cm}}
\pgfpathlineto{\pgfqpoint{1.376cm}{1.588cm}}
\pgfpathlineto{\pgfqpoint{0cm}{1.588cm}}
\pgfpathclose
\pgfusepath{clip}
\begin{pgfscope}
\begin{pgfscope}
\pgfpathmoveto{\pgfqpoint{0cm}{0cm}}
\pgfpathlineto{\pgfqpoint{1.376cm}{0cm}}
\pgfpathlineto{\pgfqpoint{1.376cm}{1.588cm}}
\pgfpathlineto{\pgfqpoint{0cm}{1.588cm}}
\pgfpathclose
\pgfusepath{clip}
\begin{pgfscope}
\begin{pgfscope}
\definecolor{eps2pgf_color}{gray}{0.976471}\pgfsetstrokecolor{eps2pgf_color}\pgfsetfillcolor{eps2pgf_color}
\pgfpathmoveto{\pgfqpoint{0cm}{0cm}}
\pgfpathlineto{\pgfqpoint{1.376cm}{0cm}}
\pgfpathlineto{\pgfqpoint{1.376cm}{1.588cm}}
\pgfpathlineto{\pgfqpoint{0cm}{1.588cm}}
\pgfpathclose
\pgfusepath{fill}
\end{pgfscope}
\begin{pgfscope}
\pgfsetdash{}{0cm}
\pgfsetlinewidth{0.818mm}
\pgfsetroundcap
\pgfsetroundjoin
\pgfsetmiterlimit{7.0}
\definecolor{eps2pgf_color}{gray}{0}\pgfsetstrokecolor{eps2pgf_color}\pgfsetfillcolor{eps2pgf_color}
\pgfpathmoveto{\pgfqpoint{0.117cm}{1.476cm}}
\pgfpathlineto{\pgfqpoint{0.682cm}{0.726cm}}
\pgfpathlineto{\pgfqpoint{1.246cm}{1.476cm}}
\pgfusepath{stroke}
\end{pgfscope}
\definecolor{eps2pgf_color}{gray}{0}\pgfsetstrokecolor{eps2pgf_color}\pgfsetfillcolor{eps2pgf_color}
\pgfpathmoveto{\pgfqpoint{0.273cm}{1.451cm}}
\pgfpathcurveto{\pgfqpoint{0.273cm}{1.487cm}}{\pgfqpoint{0.259cm}{1.522cm}}{\pgfqpoint{0.233cm}{1.547cm}}
\pgfpathcurveto{\pgfqpoint{0.207cm}{1.573cm}}{\pgfqpoint{0.173cm}{1.588cm}}{\pgfqpoint{0.137cm}{1.588cm}}
\pgfpathcurveto{\pgfqpoint{0.1cm}{1.588cm}}{\pgfqpoint{0.066cm}{1.573cm}}{\pgfqpoint{0.04cm}{1.547cm}}
\pgfpathcurveto{\pgfqpoint{0.014cm}{1.522cm}}{\pgfqpoint{0cm}{1.487cm}}{\pgfqpoint{0cm}{1.451cm}}
\pgfpathcurveto{\pgfqpoint{0cm}{1.414cm}}{\pgfqpoint{0.014cm}{1.379cm}}{\pgfqpoint{0.04cm}{1.354cm}}
\pgfpathcurveto{\pgfqpoint{0.066cm}{1.328cm}}{\pgfqpoint{0.1cm}{1.314cm}}{\pgfqpoint{0.137cm}{1.314cm}}
\pgfpathcurveto{\pgfqpoint{0.173cm}{1.314cm}}{\pgfqpoint{0.207cm}{1.328cm}}{\pgfqpoint{0.233cm}{1.354cm}}
\pgfpathcurveto{\pgfqpoint{0.259cm}{1.379cm}}{\pgfqpoint{0.273cm}{1.414cm}}{\pgfqpoint{0.273cm}{1.451cm}}
\pgfusepath{fill}
\pgfpathmoveto{\pgfqpoint{1.345cm}{1.426cm}}
\pgfpathcurveto{\pgfqpoint{1.345cm}{1.463cm}}{\pgfqpoint{1.331cm}{1.497cm}}{\pgfqpoint{1.305cm}{1.523cm}}
\pgfpathcurveto{\pgfqpoint{1.28cm}{1.549cm}}{\pgfqpoint{1.245cm}{1.563cm}}{\pgfqpoint{1.209cm}{1.563cm}}
\pgfpathcurveto{\pgfqpoint{1.172cm}{1.563cm}}{\pgfqpoint{1.138cm}{1.549cm}}{\pgfqpoint{1.112cm}{1.523cm}}
\pgfpathcurveto{\pgfqpoint{1.087cm}{1.497cm}}{\pgfqpoint{1.072cm}{1.463cm}}{\pgfqpoint{1.072cm}{1.426cm}}
\pgfpathcurveto{\pgfqpoint{1.072cm}{1.39cm}}{\pgfqpoint{1.087cm}{1.355cm}}{\pgfqpoint{1.112cm}{1.329cm}}
\pgfpathcurveto{\pgfqpoint{1.138cm}{1.304cm}}{\pgfqpoint{1.172cm}{1.289cm}}{\pgfqpoint{1.209cm}{1.289cm}}
\pgfpathcurveto{\pgfqpoint{1.245cm}{1.289cm}}{\pgfqpoint{1.28cm}{1.304cm}}{\pgfqpoint{1.305cm}{1.329cm}}
\pgfpathcurveto{\pgfqpoint{1.331cm}{1.355cm}}{\pgfqpoint{1.345cm}{1.39cm}}{\pgfqpoint{1.345cm}{1.426cm}}
\pgfusepath{fill}
\begin{pgfscope}
\pgfsetdash{}{0cm}
\pgfsetlinewidth{0.818mm}
\pgfsetroundcap
\pgfsetmiterlimit{4.0}
\pgfpathmoveto{\pgfqpoint{0.682cm}{0.726cm}}
\pgfpathlineto{\pgfqpoint{0.682cm}{0.097cm}}
\pgfusepath{stroke}
\end{pgfscope}
\end{pgfscope}
\end{pgfscope}
\end{pgfscope}
\end{tikzpicture}}} (t)\right\|_{B^{-
     \gamma}_{2, 2} (\pi _{t}\rho^a)} \lesssim \left\| (\partial_{t}\zeta )\prec X^{\!\resizebox{0.6em}{!}{
\begin{tikzpicture}
\pgfpathmoveto{\pgfqpoint{0cm}{0cm}}
\pgfpathlineto{\pgfqpoint{1.376cm}{0cm}}
\pgfpathlineto{\pgfqpoint{1.376cm}{1.588cm}}
\pgfpathlineto{\pgfqpoint{0cm}{1.588cm}}
\pgfpathclose
\pgfusepath{clip}
\begin{pgfscope}
\begin{pgfscope}
\pgfpathmoveto{\pgfqpoint{0cm}{0cm}}
\pgfpathlineto{\pgfqpoint{1.376cm}{0cm}}
\pgfpathlineto{\pgfqpoint{1.376cm}{1.588cm}}
\pgfpathlineto{\pgfqpoint{0cm}{1.588cm}}
\pgfpathclose
\pgfusepath{clip}
\begin{pgfscope}
\begin{pgfscope}
\definecolor{eps2pgf_color}{gray}{0.976471}\pgfsetstrokecolor{eps2pgf_color}\pgfsetfillcolor{eps2pgf_color}
\pgfpathmoveto{\pgfqpoint{0cm}{0cm}}
\pgfpathlineto{\pgfqpoint{1.376cm}{0cm}}
\pgfpathlineto{\pgfqpoint{1.376cm}{1.588cm}}
\pgfpathlineto{\pgfqpoint{0cm}{1.588cm}}
\pgfpathclose
\pgfusepath{fill}
\end{pgfscope}
\begin{pgfscope}
\pgfsetdash{}{0cm}
\pgfsetlinewidth{0.818mm}
\pgfsetroundcap
\pgfsetroundjoin
\pgfsetmiterlimit{7.0}
\definecolor{eps2pgf_color}{gray}{0}\pgfsetstrokecolor{eps2pgf_color}\pgfsetfillcolor{eps2pgf_color}
\pgfpathmoveto{\pgfqpoint{0.117cm}{1.476cm}}
\pgfpathlineto{\pgfqpoint{0.682cm}{0.726cm}}
\pgfpathlineto{\pgfqpoint{1.246cm}{1.476cm}}
\pgfusepath{stroke}
\end{pgfscope}
\definecolor{eps2pgf_color}{gray}{0}\pgfsetstrokecolor{eps2pgf_color}\pgfsetfillcolor{eps2pgf_color}
\pgfpathmoveto{\pgfqpoint{0.273cm}{1.451cm}}
\pgfpathcurveto{\pgfqpoint{0.273cm}{1.487cm}}{\pgfqpoint{0.259cm}{1.522cm}}{\pgfqpoint{0.233cm}{1.547cm}}
\pgfpathcurveto{\pgfqpoint{0.207cm}{1.573cm}}{\pgfqpoint{0.173cm}{1.588cm}}{\pgfqpoint{0.137cm}{1.588cm}}
\pgfpathcurveto{\pgfqpoint{0.1cm}{1.588cm}}{\pgfqpoint{0.066cm}{1.573cm}}{\pgfqpoint{0.04cm}{1.547cm}}
\pgfpathcurveto{\pgfqpoint{0.014cm}{1.522cm}}{\pgfqpoint{0cm}{1.487cm}}{\pgfqpoint{0cm}{1.451cm}}
\pgfpathcurveto{\pgfqpoint{0cm}{1.414cm}}{\pgfqpoint{0.014cm}{1.379cm}}{\pgfqpoint{0.04cm}{1.354cm}}
\pgfpathcurveto{\pgfqpoint{0.066cm}{1.328cm}}{\pgfqpoint{0.1cm}{1.314cm}}{\pgfqpoint{0.137cm}{1.314cm}}
\pgfpathcurveto{\pgfqpoint{0.173cm}{1.314cm}}{\pgfqpoint{0.207cm}{1.328cm}}{\pgfqpoint{0.233cm}{1.354cm}}
\pgfpathcurveto{\pgfqpoint{0.259cm}{1.379cm}}{\pgfqpoint{0.273cm}{1.414cm}}{\pgfqpoint{0.273cm}{1.451cm}}
\pgfusepath{fill}
\pgfpathmoveto{\pgfqpoint{1.345cm}{1.426cm}}
\pgfpathcurveto{\pgfqpoint{1.345cm}{1.463cm}}{\pgfqpoint{1.331cm}{1.497cm}}{\pgfqpoint{1.305cm}{1.523cm}}
\pgfpathcurveto{\pgfqpoint{1.28cm}{1.549cm}}{\pgfqpoint{1.245cm}{1.563cm}}{\pgfqpoint{1.209cm}{1.563cm}}
\pgfpathcurveto{\pgfqpoint{1.172cm}{1.563cm}}{\pgfqpoint{1.138cm}{1.549cm}}{\pgfqpoint{1.112cm}{1.523cm}}
\pgfpathcurveto{\pgfqpoint{1.087cm}{1.497cm}}{\pgfqpoint{1.072cm}{1.463cm}}{\pgfqpoint{1.072cm}{1.426cm}}
\pgfpathcurveto{\pgfqpoint{1.072cm}{1.39cm}}{\pgfqpoint{1.087cm}{1.355cm}}{\pgfqpoint{1.112cm}{1.329cm}}
\pgfpathcurveto{\pgfqpoint{1.138cm}{1.304cm}}{\pgfqpoint{1.172cm}{1.289cm}}{\pgfqpoint{1.209cm}{1.289cm}}
\pgfpathcurveto{\pgfqpoint{1.245cm}{1.289cm}}{\pgfqpoint{1.28cm}{1.304cm}}{\pgfqpoint{1.305cm}{1.329cm}}
\pgfpathcurveto{\pgfqpoint{1.331cm}{1.355cm}}{\pgfqpoint{1.345cm}{1.39cm}}{\pgfqpoint{1.345cm}{1.426cm}}
\pgfusepath{fill}
\begin{pgfscope}
\pgfsetdash{}{0cm}
\pgfsetlinewidth{0.818mm}
\pgfsetroundcap
\pgfsetmiterlimit{4.0}
\pgfpathmoveto{\pgfqpoint{0.682cm}{0.726cm}}
\pgfpathlineto{\pgfqpoint{0.682cm}{0.097cm}}
\pgfusepath{stroke}
\end{pgfscope}
\end{pgfscope}
\end{pgfscope}
\end{pgfscope}
\end{tikzpicture}}} (t)\right\|_{ B^{-
     \gamma}_{2, 2} (\pi _{t}\rho^a)}+\left\| (\Delta\zeta )\prec X^{\!\resizebox{0.6em}{!}{
\begin{tikzpicture}
\pgfpathmoveto{\pgfqpoint{0cm}{0cm}}
\pgfpathlineto{\pgfqpoint{1.376cm}{0cm}}
\pgfpathlineto{\pgfqpoint{1.376cm}{1.588cm}}
\pgfpathlineto{\pgfqpoint{0cm}{1.588cm}}
\pgfpathclose
\pgfusepath{clip}
\begin{pgfscope}
\begin{pgfscope}
\pgfpathmoveto{\pgfqpoint{0cm}{0cm}}
\pgfpathlineto{\pgfqpoint{1.376cm}{0cm}}
\pgfpathlineto{\pgfqpoint{1.376cm}{1.588cm}}
\pgfpathlineto{\pgfqpoint{0cm}{1.588cm}}
\pgfpathclose
\pgfusepath{clip}
\begin{pgfscope}
\begin{pgfscope}
\definecolor{eps2pgf_color}{gray}{0.976471}\pgfsetstrokecolor{eps2pgf_color}\pgfsetfillcolor{eps2pgf_color}
\pgfpathmoveto{\pgfqpoint{0cm}{0cm}}
\pgfpathlineto{\pgfqpoint{1.376cm}{0cm}}
\pgfpathlineto{\pgfqpoint{1.376cm}{1.588cm}}
\pgfpathlineto{\pgfqpoint{0cm}{1.588cm}}
\pgfpathclose
\pgfusepath{fill}
\end{pgfscope}
\begin{pgfscope}
\pgfsetdash{}{0cm}
\pgfsetlinewidth{0.818mm}
\pgfsetroundcap
\pgfsetroundjoin
\pgfsetmiterlimit{7.0}
\definecolor{eps2pgf_color}{gray}{0}\pgfsetstrokecolor{eps2pgf_color}\pgfsetfillcolor{eps2pgf_color}
\pgfpathmoveto{\pgfqpoint{0.117cm}{1.476cm}}
\pgfpathlineto{\pgfqpoint{0.682cm}{0.726cm}}
\pgfpathlineto{\pgfqpoint{1.246cm}{1.476cm}}
\pgfusepath{stroke}
\end{pgfscope}
\definecolor{eps2pgf_color}{gray}{0}\pgfsetstrokecolor{eps2pgf_color}\pgfsetfillcolor{eps2pgf_color}
\pgfpathmoveto{\pgfqpoint{0.273cm}{1.451cm}}
\pgfpathcurveto{\pgfqpoint{0.273cm}{1.487cm}}{\pgfqpoint{0.259cm}{1.522cm}}{\pgfqpoint{0.233cm}{1.547cm}}
\pgfpathcurveto{\pgfqpoint{0.207cm}{1.573cm}}{\pgfqpoint{0.173cm}{1.588cm}}{\pgfqpoint{0.137cm}{1.588cm}}
\pgfpathcurveto{\pgfqpoint{0.1cm}{1.588cm}}{\pgfqpoint{0.066cm}{1.573cm}}{\pgfqpoint{0.04cm}{1.547cm}}
\pgfpathcurveto{\pgfqpoint{0.014cm}{1.522cm}}{\pgfqpoint{0cm}{1.487cm}}{\pgfqpoint{0cm}{1.451cm}}
\pgfpathcurveto{\pgfqpoint{0cm}{1.414cm}}{\pgfqpoint{0.014cm}{1.379cm}}{\pgfqpoint{0.04cm}{1.354cm}}
\pgfpathcurveto{\pgfqpoint{0.066cm}{1.328cm}}{\pgfqpoint{0.1cm}{1.314cm}}{\pgfqpoint{0.137cm}{1.314cm}}
\pgfpathcurveto{\pgfqpoint{0.173cm}{1.314cm}}{\pgfqpoint{0.207cm}{1.328cm}}{\pgfqpoint{0.233cm}{1.354cm}}
\pgfpathcurveto{\pgfqpoint{0.259cm}{1.379cm}}{\pgfqpoint{0.273cm}{1.414cm}}{\pgfqpoint{0.273cm}{1.451cm}}
\pgfusepath{fill}
\pgfpathmoveto{\pgfqpoint{1.345cm}{1.426cm}}
\pgfpathcurveto{\pgfqpoint{1.345cm}{1.463cm}}{\pgfqpoint{1.331cm}{1.497cm}}{\pgfqpoint{1.305cm}{1.523cm}}
\pgfpathcurveto{\pgfqpoint{1.28cm}{1.549cm}}{\pgfqpoint{1.245cm}{1.563cm}}{\pgfqpoint{1.209cm}{1.563cm}}
\pgfpathcurveto{\pgfqpoint{1.172cm}{1.563cm}}{\pgfqpoint{1.138cm}{1.549cm}}{\pgfqpoint{1.112cm}{1.523cm}}
\pgfpathcurveto{\pgfqpoint{1.087cm}{1.497cm}}{\pgfqpoint{1.072cm}{1.463cm}}{\pgfqpoint{1.072cm}{1.426cm}}
\pgfpathcurveto{\pgfqpoint{1.072cm}{1.39cm}}{\pgfqpoint{1.087cm}{1.355cm}}{\pgfqpoint{1.112cm}{1.329cm}}
\pgfpathcurveto{\pgfqpoint{1.138cm}{1.304cm}}{\pgfqpoint{1.172cm}{1.289cm}}{\pgfqpoint{1.209cm}{1.289cm}}
\pgfpathcurveto{\pgfqpoint{1.245cm}{1.289cm}}{\pgfqpoint{1.28cm}{1.304cm}}{\pgfqpoint{1.305cm}{1.329cm}}
\pgfpathcurveto{\pgfqpoint{1.331cm}{1.355cm}}{\pgfqpoint{1.345cm}{1.39cm}}{\pgfqpoint{1.345cm}{1.426cm}}
\pgfusepath{fill}
\begin{pgfscope}
\pgfsetdash{}{0cm}
\pgfsetlinewidth{0.818mm}
\pgfsetroundcap
\pgfsetmiterlimit{4.0}
\pgfpathmoveto{\pgfqpoint{0.682cm}{0.726cm}}
\pgfpathlineto{\pgfqpoint{0.682cm}{0.097cm}}
\pgfusepath{stroke}
\end{pgfscope}
\end{pgfscope}
\end{pgfscope}
\end{pgfscope}
\end{tikzpicture}}}(t) \right\|_{B^{-
     \gamma}_{2, 2} (\pi _{t}\rho^a)}\]
     \[
     +\left\| (\nabla\zeta )\prec \nablaX^{\!\resizebox{0.6em}{!}{
\begin{tikzpicture}
\pgfpathmoveto{\pgfqpoint{0cm}{0cm}}
\pgfpathlineto{\pgfqpoint{1.376cm}{0cm}}
\pgfpathlineto{\pgfqpoint{1.376cm}{1.588cm}}
\pgfpathlineto{\pgfqpoint{0cm}{1.588cm}}
\pgfpathclose
\pgfusepath{clip}
\begin{pgfscope}
\begin{pgfscope}
\pgfpathmoveto{\pgfqpoint{0cm}{0cm}}
\pgfpathlineto{\pgfqpoint{1.376cm}{0cm}}
\pgfpathlineto{\pgfqpoint{1.376cm}{1.588cm}}
\pgfpathlineto{\pgfqpoint{0cm}{1.588cm}}
\pgfpathclose
\pgfusepath{clip}
\begin{pgfscope}
\begin{pgfscope}
\definecolor{eps2pgf_color}{gray}{0.976471}\pgfsetstrokecolor{eps2pgf_color}\pgfsetfillcolor{eps2pgf_color}
\pgfpathmoveto{\pgfqpoint{0cm}{0cm}}
\pgfpathlineto{\pgfqpoint{1.376cm}{0cm}}
\pgfpathlineto{\pgfqpoint{1.376cm}{1.588cm}}
\pgfpathlineto{\pgfqpoint{0cm}{1.588cm}}
\pgfpathclose
\pgfusepath{fill}
\end{pgfscope}
\begin{pgfscope}
\pgfsetdash{}{0cm}
\pgfsetlinewidth{0.818mm}
\pgfsetroundcap
\pgfsetroundjoin
\pgfsetmiterlimit{7.0}
\definecolor{eps2pgf_color}{gray}{0}\pgfsetstrokecolor{eps2pgf_color}\pgfsetfillcolor{eps2pgf_color}
\pgfpathmoveto{\pgfqpoint{0.117cm}{1.476cm}}
\pgfpathlineto{\pgfqpoint{0.682cm}{0.726cm}}
\pgfpathlineto{\pgfqpoint{1.246cm}{1.476cm}}
\pgfusepath{stroke}
\end{pgfscope}
\definecolor{eps2pgf_color}{gray}{0}\pgfsetstrokecolor{eps2pgf_color}\pgfsetfillcolor{eps2pgf_color}
\pgfpathmoveto{\pgfqpoint{0.273cm}{1.451cm}}
\pgfpathcurveto{\pgfqpoint{0.273cm}{1.487cm}}{\pgfqpoint{0.259cm}{1.522cm}}{\pgfqpoint{0.233cm}{1.547cm}}
\pgfpathcurveto{\pgfqpoint{0.207cm}{1.573cm}}{\pgfqpoint{0.173cm}{1.588cm}}{\pgfqpoint{0.137cm}{1.588cm}}
\pgfpathcurveto{\pgfqpoint{0.1cm}{1.588cm}}{\pgfqpoint{0.066cm}{1.573cm}}{\pgfqpoint{0.04cm}{1.547cm}}
\pgfpathcurveto{\pgfqpoint{0.014cm}{1.522cm}}{\pgfqpoint{0cm}{1.487cm}}{\pgfqpoint{0cm}{1.451cm}}
\pgfpathcurveto{\pgfqpoint{0cm}{1.414cm}}{\pgfqpoint{0.014cm}{1.379cm}}{\pgfqpoint{0.04cm}{1.354cm}}
\pgfpathcurveto{\pgfqpoint{0.066cm}{1.328cm}}{\pgfqpoint{0.1cm}{1.314cm}}{\pgfqpoint{0.137cm}{1.314cm}}
\pgfpathcurveto{\pgfqpoint{0.173cm}{1.314cm}}{\pgfqpoint{0.207cm}{1.328cm}}{\pgfqpoint{0.233cm}{1.354cm}}
\pgfpathcurveto{\pgfqpoint{0.259cm}{1.379cm}}{\pgfqpoint{0.273cm}{1.414cm}}{\pgfqpoint{0.273cm}{1.451cm}}
\pgfusepath{fill}
\pgfpathmoveto{\pgfqpoint{1.345cm}{1.426cm}}
\pgfpathcurveto{\pgfqpoint{1.345cm}{1.463cm}}{\pgfqpoint{1.331cm}{1.497cm}}{\pgfqpoint{1.305cm}{1.523cm}}
\pgfpathcurveto{\pgfqpoint{1.28cm}{1.549cm}}{\pgfqpoint{1.245cm}{1.563cm}}{\pgfqpoint{1.209cm}{1.563cm}}
\pgfpathcurveto{\pgfqpoint{1.172cm}{1.563cm}}{\pgfqpoint{1.138cm}{1.549cm}}{\pgfqpoint{1.112cm}{1.523cm}}
\pgfpathcurveto{\pgfqpoint{1.087cm}{1.497cm}}{\pgfqpoint{1.072cm}{1.463cm}}{\pgfqpoint{1.072cm}{1.426cm}}
\pgfpathcurveto{\pgfqpoint{1.072cm}{1.39cm}}{\pgfqpoint{1.087cm}{1.355cm}}{\pgfqpoint{1.112cm}{1.329cm}}
\pgfpathcurveto{\pgfqpoint{1.138cm}{1.304cm}}{\pgfqpoint{1.172cm}{1.289cm}}{\pgfqpoint{1.209cm}{1.289cm}}
\pgfpathcurveto{\pgfqpoint{1.245cm}{1.289cm}}{\pgfqpoint{1.28cm}{1.304cm}}{\pgfqpoint{1.305cm}{1.329cm}}
\pgfpathcurveto{\pgfqpoint{1.331cm}{1.355cm}}{\pgfqpoint{1.345cm}{1.39cm}}{\pgfqpoint{1.345cm}{1.426cm}}
\pgfusepath{fill}
\begin{pgfscope}
\pgfsetdash{}{0cm}
\pgfsetlinewidth{0.818mm}
\pgfsetroundcap
\pgfsetmiterlimit{4.0}
\pgfpathmoveto{\pgfqpoint{0.682cm}{0.726cm}}
\pgfpathlineto{\pgfqpoint{0.682cm}{0.097cm}}
\pgfusepath{stroke}
\end{pgfscope}
\end{pgfscope}
\end{pgfscope}
\end{pgfscope}
\end{tikzpicture}}} (t)\right\|_{ B^{-
     \gamma}_{2, 2} (\pi_{t} \rho^a)}\lesssim \| \partial_{t}\zeta (t)\|_{ B^{-1-\beta}_{2,2} (\pi_{t} \rho^{c})} + \| \zeta (t)\|_{ B^{1 -
     \beta}_{2, 2} (\pi_{t})}, \]
  \[ \| \Upsilon \zeta (t) \|_{B^{- \gamma}_{2, 2} (\pi_{t} \rho^a)} \lesssim \|
     \zeta (t) \|_{B^{1 - \beta}_{2, 2} (\pi_{t})} . \]
     We used the fact that $c<a$ to estimate the commutator above. 
  Besides, note that $\LL \zeta$ contains only one term which does not appear
  in $\LL \eta$ and it does not contain any term which requires time regularity. The additional term is controlled by
  \[ \| 3 \llbracket X^2 \rrbracket \succ \zeta (t) \|_{B^{- 1 - \kappa}_{2,
     2} (\pi_{t} \rho^c)} \lesssim \| \zeta (t) \|_{B^{1 - \beta}_{2, 2} (\pi_{t})} .
  \]
Using the equation for $\zeta$ we get
  \[ \| \partial_t \zeta(t) \|_{ B^{- 1 - \beta}_{2, 2} (\pi_{t} \rho^c)}
     \lesssim \left\| (\mu - \Delta) \zeta (t) - \LL \zeta (t) \right\|_{
     B^{- 1 - \beta}_{2, 2} (\rho^c \pi_{t})} \]
  \[ \lesssim \| \zeta (t)\|_{B^{1 - \beta}_{2, 2} (\pi_{t})} + \left\| \LL
     \zeta (t) \right\|_{B^{- 1 - \kappa}_{2, 2} (\pi_{t} \rho^c)}\lesssim \| \zeta(t) \|_{ B^{1 - \beta}_{2, 2} (\pi_{t})}  + \| \eta(t)
     \|_{ B^{1 + \beta}_{2, 2} (\pi_{t})}, \]
   where we applied the above estimates for $\LL \zeta$ again. Therefore, it
  follows
  \[ \left\| \LL \eta(t) \right\|_{ B^{- \gamma}_{2, 2} (\pi _{t}\rho^a)} +
     \left\| \LL \zeta (t)\right\|_{ B^{- 1 - \kappa}_{2, 2} (\pi_{t} \rho^c)}
     \lesssim \|\zeta(t) \|_{B^{1 - \beta}_{2, 2} (\pi_{t})} + \| \eta(t)
     \|_{ B^{1 + \beta}_{2, 2} (\pi_{t})} . \]
  Finally we have all in hand to conclude. Choosing $\delta$ sufficiently
  small allows to absorb the two terms in
  (\ref{eq:s}) into the left hand side and by Gronwall's lemma we absorb the terms in \eqref{eq:s1}. Accordingly $\zeta \equiv \eta \equiv
  0$ on $[0, T]$ and the proof is complete.
\end{proof}

\section{Coming down from infinity}
\label{s:d}

The goal of this section is to establish refined a priori estimates for solutions to the parabolic $\Phi^{4}$ model in dimension 2 and 3, which are valid independently of the initial condition. In particular this shows  that the solutions come down from infinity.

For the purposes of this section, we introduce a further time weight of the form $\tau (t) = 1 - e^{- t}$ so
that $\partial_t \tau (t) =  e^{- t} = 1 - \tau (t)$.
First of all, we prepare the initial data as follows. Let $\varphi_{0}\in \CC^{ - 1 + \varepsilon}
     (\rho^{1+\varepsilon}_0)$ for some $\varepsilon>0$.  Let $L>0$ be such that 
  \[ 2^{\varepsilon L} \sim \| \varphi_0 \|_{\CC^{ - 1 + \varepsilon}
     (\rho^{1+\varepsilon}_0)} . \]
Define 
$$\phi_{0}:=\UU_{>}
 \varphi_{0}-X(0),\qquad \psi_{0}:=\UU_{\leq}\varphi_{0},$$
 where $\UU_{>},\UU_{\leq}$ are the localizers corresponding to $L$. We recall that $X$ was chosen stationary and $X(0)\in \CC^{-1/2-\kappa}(\rho^{\sigma}_{0})$ for any $\sigma>0$  (see Theorem \ref{thm:renorm43}). Then it follows from Lemma \ref{lem:local} that
  \begin{equation}\label{eq:phi0}
  \|  \phi_0 \|_{\CC^{ - 1} (\rho_0)} \lesssim 1 
  \end{equation}
  uniformly over $\varphi_0 \in \CC^{ - 1 + \varepsilon} (\rho^{1+\varepsilon}_0)$ and $\varepsilon>0$.
Now we have all in hand to formulate the main result of this section.

\begin{theorem}
  \label{th:come-down-3}
  Let $\kappa, \alpha\in (0,1)$ be chosen sufficiently small and let $\gamma=\alpha-\kappa>0$. Let $\varphi_{0}\in \CC^{ - 1 + \varepsilon}
     (\rho^{1+\varepsilon}_0)$ for some $\varepsilon>0$. Let $(\phi, \psi)$ be a solution to the parabolic problem \eqref{eq:two43}, \eqref{eq:th43}
  in $d = 3$ with initial condition $(\phi_{0},\psi_{0})$ defined above.  
  Then, uniformly in $\varphi_{0}$,
  \[ \phi \in C \CC^{\alpha} (\tau^{\frac{1}{2}} \rho) \cap C \CC^{\frac{1}{2}
     + \alpha} ((\tau^{\frac{1}{2}} \rho)^{\frac{3}{2} + \alpha}), \]
  and
  \[ \psi \in C \CC^{2 + \gamma} ((\tau^{\frac{1}{2}} \rho)^{3 + \gamma}) \cap
     L^{\infty} L^{\infty} (\tau^{\frac{1}{2}} \rho) . \]
\end{theorem}

An analogous result holds true also in dimension 2. In this case, we construct the initial condition in the same way and obtain the following result.

\begin{theorem}
  \label{th:come-down-2}
  Let $\kappa, \alpha\in (0,1)$ be chosen sufficiently small and let $\beta=\alpha-\kappa>0$. Let $(\phi, \psi)$ a solution of the parabolic problem \eqref{eq:42b}
  in $d = 2$ with initial condition $(\phi_{0},\psi_{0})$ defined above. Then
  \[ (\phi, \psi) \in C \CC^{\alpha} (\tau^{\frac 12} \rho) \times [C \CC^{2
     + \beta} ((\tau^{ \frac 12} \rho)^{3 + \beta}) \cap L^{\infty}
     L^{\infty} (\tau^{\frac 12} \rho)] \]
  uniformly in the initial condition.
\end{theorem}

In the following, we discuss the necessary preliminary results and finally prove Theorem~\ref{th:come-down-3} in Section \ref{ss:thm91}. The proof of
Theorem~\ref{th:come-down-2} will not be given since it is substantially simpler
(as one does not need the paracontrolled ansatz) and follows the same pattern.
As a corollary,  we obtain that
\[ \| \phi (t) + \psi (t) \|_{L^{\infty}(\rho)} \lesssim 1 + t^{- 1 / 2} \]
independently of the initial condition: the solution comes down from infinity in a
finite time. This has been first observed by~\cite{MWcomedown} in the periodic setting.

\begin{remark}
We point out that the assumption on the regularity of initial condition in Theorems \ref{th:come-down-3}, \ref{th:come-down-2} is very weak and the existence for such singular initial conditions is not guaranteed  by the respective existence results, Theorems \ref{thm:ex3d}, \ref{thm:ex42}. However, for instance in case of the $\Phi^{4}$ model  \eqref{eq:phi4p} on $\mathbb{T}^{3}$, if $\varphi_{0}$ belongs to the natural space $\CC^{-1/2-\kappa}$, one may use the short time existence of a unique solution from \cite{CC} together with Theorem \ref{th:come-down-3} to deduce global existence and the coming down from infinity property. Furthermore, revisiting the proof of our a~priori estimates we see that the proportionality constants depend polynomially on the noise, which implies integrability of all the moments. This way, we recover the result of~\cite{MWcomedown}.
\end{remark}

\subsection{Interpolation and localization}
\label{ss:d0}

First, we notice that an interpolation similar to Lemma \ref{lemma:interp2} remains valid and the proof follows the same lines.

\begin{lemma}
  \label{lemma:tau-space-interpolation}For $\alpha \in [0, 2 + \kappa]$ we
  have
  \[ \| \tau^{ (1 + \alpha) / 2} \psi \|_{C \CC^{\alpha} (\rho^{1 + \alpha})}
     \lesssim \| \tau^{1 / 2} \psi \|_{C L^{\infty} (\rho)}^{1 - \alpha / (2 + 
     \kappa)} \| \tau^{(3 + \kappa) / 2} \psi \|_{C \CC^{2 + \kappa} (\rho^{3
     + \kappa})}^{\alpha / (2 + \kappa)}, \]
  or more generally
  \[ \| \tau^{ (1 + \alpha) / 2} \psi \|_{C \CC^{\alpha} (\rho^{1 + \alpha})}
     \lesssim \| \tau^{(1 + \delta) / 2} \psi \|_{C \CC^{\delta} (\rho^{1 +
     \delta})}^{1 - \theta} \| \tau^{(3 + \kappa) / 2} \psi \|_{C \CC^{2 +
     \kappa} (\rho^{3 + \kappa})}^{\theta}, \]
  whenever $\delta \in [0, \alpha]$ and $\alpha = (1 - \theta) \delta + \theta
  (2 + \kappa)$ for some $\theta \in [0, 1]$.
\end{lemma}

We stress that unlike Lemma \ref{lemma:interp2}, we do not include any interpolation in terms of time regularity into Lemma \ref{lemma:tau-space-interpolation}. Indeed, since the weight $\tau$ vanishes at zero, the equivalence  \eqref{eq:teq} is not valid anymore for the  corresponding weighted H\"older  spaces (in time). Therefore we proceed differently than in  Section \ref{sec:43}: below, we introduce a new modified paracontrolled ansatz which eventually leads us to the requirement $\tau^{\beta}(\phi+\psi)\in C^{{\delta}}L^{\infty}(\rho^{\sigma})$, for certain $\beta,\delta,\sigma>0$. In other words, instead of time regularity of $\phi+\psi$ in a space weighted by $\tau^{\beta}\rho^{\sigma}$, we require time regularity of $\tau^{\beta}(\phi+\psi)$ in a space weighted by an admissible space-time weight $\rho^{\sigma}$, which falls in the framework of Section \ref{ssec:besov}.

We will also need the parabolic localization \eqref{eq:locd} together with Lemma \ref{lem:local2}. However, since this is only applied to the stochastic objects that do not require any $\tau$ weight, no $\tau$-adapted version of Lemma \ref{lem:local2} is needed.

\subsection{Weighted Schauder estimates}
\label{ss:d1}

In this section we formulate new Schauder estimates  adapted to the particular weight $\tau$ which is not bounded away from zero. 
In particular, with the interpolation in hand, we may  employ Lemma~\ref{lemma:basic-tau-schauder} to deduce the following.

\begin{lemma}
  \label{lemma:tau-weighted-schauder}
   Let $\alpha
  > - 2$, $\gamma = (3 + \alpha) / 2$ and $\beta_i \in [0, 2)$. Assume that
  $\LL v = \sum_i V_i$. Then the following a priori estimate holds true
  \[ \| v \|_{C \CC^{2 + \alpha} (\tau^{\gamma} \rho)} \lesssim \| v \|_{C
     \CC^{\delta} (\tau^{(1 + \delta) / 2} \rho)} + \sum_i \| \tau^{\beta_i /
     2} V_i \|_{C \CC^{\alpha + \beta_i} (\tau^{\gamma} \rho)}, \]
  whenever $\delta \in [0, 1 + \alpha]$ is given by $1 + \alpha = (1 - \theta)
  \delta + \theta (2 + \alpha)$ for some $\theta \in [0, 1]$. Consequently, it
  also holds
  \[ \| v \|_{C \CC^{2 + \alpha} (\tau^{\gamma} \rho)} \lesssim \| v
     \|_{L^{\infty} L^{\infty} (\tau^{1 / 2} \rho)} + \sum_i \| \tau^{\beta_i
     / 2} V_i \|_{C \CC^{\alpha + \beta_i} (\tau^{\gamma} \rho)} . \]
\end{lemma}

\begin{proof}
First observe that
  \[ (\partial_t + \mu - \Delta) (\tau^{\gamma} \rho v) = \tau^{\gamma} \rho \sum_i V_i
     - \left( \frac{\Delta \rho}{\rho} \right) \tau^{\gamma} \rho v - 2
     \frac{\nabla \rho}{\rho} \nabla (\tau^{\gamma} \rho v) + 2 \left(
     \frac{\nabla \rho}{\rho} \right)^2 \tau^{\gamma} \rho v  \]
  \[ + \left(
     \frac{\partial_t \rho}{\rho} \right) \tau^{\gamma} \rho v+ \gamma \tau^{- 1} (1 - \tau) \tau^{\gamma} \rho v. \]
\rmbb{  By a slight modification of Lemma~\ref{lemma:basic-tau-schauder} where we choose different $\beta$ for different terms, we deduce
  $$
  \| \tau^{\gamma} \rho v \|_{C \CC^{2 + \alpha}}\lesssim \sum_i \| \tau^{\beta_i /
     2} V_i \|_{C \CC^{\alpha + \beta_i} (\tau^{\gamma} \rho)} + \| \gamma \tau^{- 1} (1 - \tau) \tau^{\gamma} \rho v
     \|_{C \CC^{2 +
     \alpha'} (\tau^{(2 - \alpha + \alpha') / 2})}
 $$
 $$  
 + \left\| -\left( \frac{\Delta \rho}{\rho} \right) \tau^{\gamma} \rho v - 2
     \frac{\nabla \rho}{\rho} \nabla (\tau^{\gamma} \rho v) + 2 \left(
     \frac{\nabla \rho}{\rho} \right)^2 \tau^{\gamma} \rho v +\left(
     \frac{\partial_t \rho}{\rho} \right) \tau^{\gamma} \rho v \right\|_{C \CC^{ \alpha}} 
  $$}
  \[  \lesssim  \sum_i \| \tau^{\beta_i /
     2} V_i \|_{C \CC^{\alpha + \beta_i} (\tau^{\gamma} \rho)} + \| \tau^{\gamma - 1} \rho v \|_{C \CC^{2 +
     \alpha'} (\tau^{(2 - \alpha + \alpha') / 2})} + \| \tau^{\gamma} \rho v
     \|_{C \CC^{1 + \alpha}}, \]
  where $\alpha < 2 + \alpha' < 2 + \alpha$. Now using  $\tau^{(3+\alpha)/2}\leq \tau^{(2+\alpha)/2}$ for the last term on the right hand side, we obtain by interpolation from
  Lemma~\ref{lemma:tau-space-interpolation} that 
  \[ \| v \|_{C \CC^{1 + \alpha} (\tau^{(3 + \alpha) / 2} \rho)} + \| v \|_{C
     \CC^{2 + \alpha'} (\tau^{(3 + \alpha') / 2} \rho)} \leqslant C \| v \|_{C
     \CC^{\delta} (\tau^{(1 + \delta) / 2} \rho)} + \frac{1}{2} \| v \|_{C
     \CC^{2 + \alpha} (\tau^{(3 + \alpha) / 2} \rho)} . \]
  This allows us to absorb the residual terms in the left hand side giving the
  final statement.
\end{proof}

Below we need also some specific Schauder estimate for time regularity of
solutions to the heat equation with a precise control of the $\tau$-weights in
the source term. We derive it here.

\begin{lemma}
  \label{lemma:time-tau-schauder}For any $\alpha \in (0, 2)$ and $\beta_i \in
  [0, 2)$ such that $\alpha + \beta_i - 2 < 0$ we have
  \[ \| v \|_{C^{\alpha / 2} L^{\infty} (\rho)} \lesssim \| v \|_{C
     \CC^{\alpha} (\rho)} + \sum_i \| \tau^{\beta_i / 2} V_i \|_{C \CC^{\alpha
     + \beta_i - 2} (\rho)} \]
  where \ $\LL v = \sum_i V_i$.
\end{lemma}

\begin{proof}
  Let $f = (\partial_t + \mu - \Delta) v$ and recall that we denoted by $P_{t}=e^{t(\Delta-\mu)}$ the  semigroup of operators generated by $\Delta-\mu$ with $\mu>0$. Fix $t>s \geqslant t-1$ and let $k \in \mathbbm{N}_0$ be such that $2^{- 2 k} \sim | t
  - s |$. Then using the fact that the weight $\rho$ is nonincreasing in time we obtain
  \[ \|v (t) - v (s) \|_{L^{\infty} (\rho_{t})} \lesssim \| \Delta_{\leqslant k}
     (P_{t - s} - \tmop{Id}) v (s) \|_{L^{\infty} (\rho_{s})} \]
  \[ + \int_s^t \| \Delta_{\leqslant k} P_{t - u} f (u) \|_{L^{\infty} (\rho_{u})}
     \mathd u +  \|\Delta_{>k} (v(t)-v(s)) \|_{ L^{\infty} (\rho_{t})} \backassign I_1
     + I_2 + I_3 \]
  Now
  $$
  I_{3}\lesssim \|\Delta_{>k}v(t) \|_{L^{\infty} (\rho_{t})} +\|\Delta_{>k}v(s) \|_{L^{\infty} (\rho_{s})} \lesssim 2^{- \alpha k} \| v \|_{C \CC^{\alpha} (\rho)},
  $$
  \[ I_1 \lesssim | t - s |^{\alpha / 2} \| v \|_{C \CC^{\alpha} (\rho)}, \]
  and if $t\le 2$ then
  \[ I_2 \lesssim \sum_i 2^{(2 - \alpha - \beta_i) k} \int_s^t \tau (u)^{-
     \beta_i / 2} \mathd u \| \tau^{\beta_i / 2} V_i \|_{C \CC^{\alpha +
     \beta_i - 2} (\rho)} \]
  \[ \lesssim \sum_i 2^{(2 - \alpha - \beta_i) k} 
     \int_s^t u^{- \beta_i / 2} \mathd u \| \tau^{\beta_i / 2} V_i \|_{C
     \CC^{\alpha + \beta_i - 2} (\rho)} \]
  \[ \lesssim | t - s |^{\alpha / 2} \left[ \sum_i \frac{(t^{1 - \beta_i / 2}
     - s^{1 - \beta_i / 2})}{| t - s |^{(1 - \beta_i / 2)}} \| \tau^{\beta_i /
     2} V_i \|_{C \CC^{\alpha + \beta_i - 2} (\rho)} \right], \]
Since the function $t\mapsto t^{\delta}$ is $\delta$-H\"older continuous, it holds true
  \[ \frac{(t^{1 - \beta_i / 2} - s^{1 - \beta_i / 2})}{| t - s |^{(1 -
     \beta_i / 2)}} \lesssim 1, \]
  hence
  \[ I_2 \lesssim | t - s |^{\alpha / 2} \sum_i \| \tau^{\beta_i / 2} V_i
     \|_{C \CC^{\alpha + \beta_i - 2} (\rho)} . \]
  If $t>2$ then $s\geq 1$ and consequently
  \[ I_2 \lesssim \sum_i 2^{(2 - \alpha - \beta_i) k} (t - s) \| \tau^{\beta_i
     / 2} V_i \|_{C \CC^{\alpha + \beta_i - 2} (\rho)} \]
  \[ \lesssim | t - s |^{\alpha / 2} \left[ \sum_i \frac{(t^{} - s^{})}{| t -
     s |^{(1 - \beta_i / 2)}} \| \tau^{\beta_i / 2} V_i \|_{C \CC^{\alpha +
     \beta_i - 2} (\rho)} \right] \lesssim | t - s |^{\alpha / 2} \sum_i \|
     \tau^{\beta_i / 2} V_i \|_{C \CC^{\alpha + \beta_i - 2} (\rho)} . \]
Hence we can conclude that
  \[ \sup_{\substack{s,t\in [0,\8)\\t-1\leq s < t}} \frac{\|v (t) - v (s) \|_{L^{\infty} (\rho_{t})}}{| t - s
     |^{\alpha / 2}} \lesssim \sum_i \| \tau^{\beta_i / 2} V_i \|_{C
     \CC^{\alpha + \beta_i - 2} (\rho)} + \| v \|_{C \CC^{\alpha} (\rho)} \]
and the claim follows.
\end{proof}

Similarly to Lemma \ref{lemma:schauder-par} we finally  derive a Schauder estimate for equations including a cubic nonlinearity.

\begin{lemma}
  \label{lemma:schauder-par1}
Fix $\kappa > 0$ and let $\psi \in C\CC^{2 +
  \kappa} (\rho^{3 + \kappa}) \cap C^{1}L^{\infty} (\rho^{3+\kappa})\cap L^{\infty}L^{\infty}(\rho)$ be a classical solution
  to
  \[ \partial_{t}\psi+(- \Delta + \mu) \psi + \psi^3 = \Psi ,\qquad \psi(0)=\psi_{0}.\]
Then
  \[ \| \psi \|_{C\CC^{2 + \kappa} ((\tau^{1/2}\rho)^{3 + \kappa})} \lesssim  1+ \|
     \Psi \|_{C\CC^{\kappa} ((\tau^{1/2}\rho)^{3 + \kappa})} + \| \psi \|_{L^{\infty}L^{\infty}
     (\tau^{1/2}\rho)}^{3 + \kappa}. \]
\end{lemma}

\begin{proof}
The proof follows from Lemma \ref{lemma:tau-space-interpolation} and Lemma \ref{lemma:tau-weighted-schauder} using the same approach as in the proof of Lemma \ref{lemma:schauder-par}.
\end{proof}

\subsection{Weighted  coercive estimate}
\label{ss:d2}

Next, we show that also the coercive estimates remain valid for the weight $\tau$, the proof uses the same ideas as Lemma \ref{lemma:apriori-parabolic}.

\begin{lemma}
  \label{lemma:weighted-coercive}
  Fix $\kappa > 0$ and let $\psi \in C\CC^{2 +
  \kappa} (\rho^{3 + \kappa}) \cap C^{1}L^{\infty} (\rho^{3+\kappa})\cap L^{\infty}L^{\infty}(\rho)$ be a classical solution
  to
  \[ \partial_t \psi + (\mu - \Delta) \psi + \psi^3 = \Psi, \qquad \psi (0) =
     \psi_0 . \]
  Then the following a priori estimate holds
  \[ \| \psi \|_{L^{\infty} L^{\infty} (\tau^{1 / 2} \rho)} \lesssim 1+ \|
     \Psi \|^{1 / 3}_{L^{\infty} L^{\infty} (\tau^{3 / 2} \rho^3)}  \]
     independently of the initial condition.
\end{lemma}

\begin{proof}
  Let $\bar{\psi} = \tau^{1 / 2} \psi \rho$ and assume for the moment that
  $\bar{\psi}$ attains its (global) maximum $M = \bar{\psi} (t^{\ast},
  x^{\ast})$ at the point $(t^{\ast}, x^{\ast})$. If $M \leq 0$, then it is
  necessary to investigate the minimum point (or alternatively the maximum of
  $- \bar{\psi}$), which we discuss below. Let us therefore assume that $M >
  0$. Then necessarily $t^{\ast} > 0$ since $\bar{\psi} (0) = 0$ and
  \[ \tau \rho^2 \partial_t  \bar{\psi} + \tau \rho^2 (- \Delta + \mu)
     \bar{\psi} + \bar{\psi}^3 = \tau^{3 / 2} \rho^3 \Psi + 
     (\tau \rho \partial_t \rho + \rho^2 \tau^{1 / 2} \partial_t \tau^{1 / 2})
     \bar{\psi} - \tau^{3 / 2} \rho^2  (\Delta \rho) \psi - 2 \tau^{3 / 2}
     \rho^2 \nabla \rho \nabla \psi . \]
  and
  \[ \partial_t  \bar{\psi} (t^{\ast}, x^{\ast}) = 0, \qquad \nabla \bar{\psi}
     (t^{\ast}, x^{\ast}) = 0, \qquad \Delta \bar{\psi} (t^{\ast}, x^{\ast})
     \leq 0 \]
  hence $\rho \nabla \psi = - \psi \nabla \rho$. Consequently $- \rho^2 \Delta
  \bar{\psi} (t^{\ast}, x^{\ast}) \geq 0$ and also $\rho \partial_t \rho
  \bar{\psi} (t^{\ast}, x^{\ast}) \leqslant 0$ since $\partial_t \rho
  \leqslant 0$. Hence
  \[ M^3 \leqslant [\tau^{3 / 2} \rho^3 \Psi ]_{|_{(t^{\ast}, x^{\ast})}} + [-\mu \tau\rho^2+\rho^2-\tau\rho\Delta\rho+2\tau |\nabla\rho|^{2}]_{|_{(t^{\ast}, x^{\ast})}} M \]
  \[ \leqslant \| \Psi \|_{L^{\infty} L^{\infty} (\tau^{3 / 2} \rho^3)} +
     c_{\rho, \mu} \| \bar{\psi} \|_{L^{\infty} L^{\infty}} . \]
  Therefore we deduce that
  \[ \bar{\psi} (t^{*},x^{*})= M \lesssim \| \Psi \|^{1 / 3}_{L^{\infty} L^{\infty}
     (\tau^{3 / 2} \rho^3)} + c_{\rho, \mu} \| \psi \|^{1 / 3}_{L^{\infty}
     L^{\infty} (\tau^{1 / 2} \rho)} . \]
  If $M < 0$ we can apply the same argument to $- \bar{\psi}$ to get
  \[ - \bar{\psi} \lesssim \| \Psi \|^{1 / 3}_{L^{\infty} L^{\infty} (\tau^{3
     / 2} \rho^3)} + c_{\rho, \mu} \| \psi \|^{1 / 3}_{L^{\infty} L^{\infty}
     (\tau^{1 / 2} \rho)} . \]
  hence
  \[ \| \psi \|_{L^{\infty} L^{\infty} (\tau^{1 / 2} \rho)} \lesssim \| \Psi
     \|^{1 / 3}_{L^{\infty} L^{\infty} (\tau^{3 / 2} \rho^3)} + c_{\rho, \mu}
     \| \psi \|^{1 / 3}_{L^{\infty} L^{\infty} (\tau^{1 / 2} \rho)} . \]
  and applying the weighted Young inequality yields the claim.

The conclusion in the case when $\tau^{1 / 2} \rho \psi$ does not
  attain its global maximum follows the same argument as in Lemma \ref{lemma:apriori-parabolic}.
\end{proof}

\begin{remark}\label{rem:down}
We note that the choice of $1/2$ as the power of the weight $\tau$ is dictated by the cubic nonlinearity through Lemma \ref{lemma:weighted-coercive}. More precisely, taking $\tau^{\alpha}$ instead of $\tau^{1/2}$ for some $\alpha>0$ and repeating the proof of the above maximum principle, it is necessary to control $\tau^{\alpha}\partial_{t}\tau^{\alpha}$ which leads to the condition $\alpha\geq 1/2$. Hence the choice $\alpha=1/2$ gives the best result.
\end{remark}

\subsection{Proof of Theorem \ref{th:come-down-3}}
\label{ss:thm91}

\begin{proof}
  Using the approach of Section \ref{sec:45} while choosing localization
  operators \eqref{eq:locd} according to the weighted norm $\| \phi + \psi \|_{L^{\infty}
  L^{\infty} (\tau^{1 / 2} \rho)}$ and the constant $L$ of various objects as in Table \ref{t:loc}, we obtain
     \begin{equation}\label{eq:Ph}
 \| \Phi \|_{C \CC^{- 2 + \alpha} (\tau^{(1 + \alpha) / 2} \rho)} \lesssim
     1.
     \end{equation}
The choice of the weight above is due to the following application of the weighted Schauder estimates from Lemma~\ref{lemma:tau-weighted-schauder}.   Namely, Lemma~\ref{lemma:tau-weighted-schauder} implies
  \begin{equation}
    \| \phi \|_{C \CC^{\alpha} (\tau^{(1 + \alpha) / 2} \rho^{})} \lesssim \|
    \phi \|_{L^{\infty} L^{\infty} (\tau^{1 / 2} \rho)} + \| \Phi \|_{C \CC^{-
    2 + \alpha} (\tau^{(1 + \alpha) / 2} \rho)} . \label{eq:down-phi}
  \end{equation}
  We note that this is not yet sufficient to bound $\phi$ in $L^{\infty}
  L^{\infty} (\tau^{1 / 2} \rho)$, since the power of $\tau$ on the left hand side of
  (\ref{eq:down-phi}) is bigger than $1 / 2$, which matters for $t$ small. In order to fill this gap, let $t\in (0,1)$ and
  $k \in \mathbbm{N}$ be such that $2^{- 2 k} \sim \lambda^{2 / \alpha} \tau (t)$
  for some small $\lambda > 0$ to be chosen below (the proportionality constants do not depend on time). Write
  \[ \| \tau^{(1 + \alpha) / 2} \phi (t) \|_{L^{\infty} (\rho_t)} \leqslant \|
     \tau^{(1 + \alpha) / 2} \Delta_{\leqslant k} \phi (t) \|_{L^{\infty}
     (\rho_t)} + \| \tau^{(1 + \alpha) / 2} \Delta_{> k} \phi (t)
     \|_{L^{\infty} (\rho_t)}, \]
  where $\rho_t (\cdot)$ denotes the weight $\rho (t, \cdot)$. For the first
  term on the right hand side we use the equation satisfied by $\phi$ together
  with the fact that the weight $\rho$ is nonincreasing in time. This gives
  \[ \| \tau^{(1 + \alpha) / 2} \phi (t) \|_{L^{\infty} (\rho_t)} \lesssim
     \tau^{(1 + \alpha) / 2} (t) \| \Delta_{\leqslant k} P_t \phi (0)
     \|_{L^{\infty} (\rho_0)} \]
  \[ + \tau^{(1 + \alpha) / 2} (t) \int_0^t  \|
     \Delta_{\leqslant k} P_{t - s} \Phi_{} (s) \|_{L^{\infty} (\rho_s)}
     \mathd s + \tau^{(1 + \alpha) / 2} (t) \| \Delta_{> k} \phi (t)
     \|_{L^{\infty} (\rho_t)} . \]
Recall  that $P_{t}=e^{t(\Delta-\mu)}$. Hence in view of \eqref{eq:sch10}, the definition of $\tau$ and the fact that $t\in(0,1)$, $\alpha - 2 < 0$ and $\alpha > 0$,  the above is further estimated
  by
  \[ \lesssim \tau^{(1 + \alpha) / 2} (t) 2^{k} \| \phi (0)
     \|_{\CC^{- 1} (\rho_0)} + 2^{- (\alpha - 2) k} \tau (
     t)\| \tau^{(1 + \alpha) / 2} \Phi_{} \|_{C \CC^{\alpha - 2}
     (\rho)} \]
  \[ + 2^{- \alpha k} \| \tau^{(1 + \alpha) / 2} \phi \|_{C \CC^{\alpha}
     (\rho)} . \]
  Using the definition of $k$ we therefore obtain
  \[ \| \tau^{(1 + \alpha) / 2} \phi (t) \|_{L^{\infty} (\rho_t)} \lesssim
     \tau (t)^{\alpha / 2} \frac{\tau (t)^{1 / 2}}{\tau (t)^{1 / 2}} \| \phi
     (0) \|_{\CC^{ - 1} (\rho_0)} + \tau (t)^{\alpha / 2} \frac{\tau
     (t)}{ \tau (t)} \| \tau^{(1 + \alpha) / 2} \Phi_{} \|_{C
     \CC^{\alpha - 2} (\rho)} \]
  \[ + \lambda \tau^{\alpha / 2} \| \tau^{(1 + \alpha) / 2} \phi \|_{C
     \CC^{\alpha} (\rho)}. \]
Hence we may divide  by $\tau^{\alpha/2}$ to obtain
  \[ \| \tau^{1 / 2} \phi (t) \|_{L^{\infty} (\rho_t)} \lesssim_{} \| \phi (0)
     \|_{\CC^{ - 1} (\rho_0)} + \| \tau^{(1 + \alpha) / 2} \Phi_{} \|_{C
     \CC^{\alpha - 2} (\rho)} + \lambda \| \tau^{(1 + \alpha) / 2} \phi \|_{C
     \CC^{\alpha} (\rho)} . \]
  Taking supremum in time and applying (\ref{eq:down-phi}) leads to
  \[ \| \phi \|_{L^{\infty} L^{\infty} (\tau^{1 / 2} \rho)} \lesssim_{} \|
     \phi (0) \|_{\CC^{ - 1} (\rho_0)} + \lambda^{} \| \phi
     \|_{L^{\infty} L^{\infty} (\tau^{1 / 2} \rho)} + \| \Phi \|_{C \CC^{- 2 +
     \alpha} (\tau^{(1 + \alpha) / 2} \rho)}. \]
Hence we can absorb the second term on the right hand side into the left hand side by choosing $\lambda$ sufficiently small.
Therefore, according to our construction of the initial datum $\phi(0)$, namely due to \eqref{eq:phi0}, and \eqref{eq:Ph}, \eqref{eq:down-phi}   we obtain
  \[ \| \phi \|_{L^{\infty} L^{\infty} (\tau^{1 / 2} \rho)} + \| \phi \|_{C
     \CC^{\alpha} (\tau^{(1 + \alpha) / 2} \rho^{})} \lesssim 1\]
     uniformly in the initial condition.
  Next, we apply Lemma \ref{lemma:tau-weighted-schauder} again, use the above
  $L^{\infty}$-bound for $\phi$ together with estimates similar to Section
  \ref{ssec:phi2} to obtain for some $\varepsilon\in (0,1)$
  \begin{equation}\label{eq:ph-down}
   \| \phi \|_{C \CC^{1 / 2 + \alpha} (\tau^{3 / 4 + \alpha / 2} \rho^{3 / 2
     + \alpha})} \lesssim 1 + \| \Phi \|_{C \CC^{- 3 / 2 + \alpha} (\tau^{3 /
     4 + \alpha / 2} \rho^{3 / 2 + \alpha})} \lesssim 1 + \| \psi
     \|_{L^{\infty} L^{\infty} (\tau^{1 / 2} \rho)}^{\varepsilon} .
     \end{equation}

As in Section \ref{sec:43}, the next step is a paracontrolled ansatz
  for $\phi .$ Here we have to be more careful due to the weight $\tau$. More
  precisely, we introduce a modified paracontrolled ansatz according to the
  formula
  \begin{equation}\label{eq:th2}
   \phi = \breve{\vartheta} - \tau^{- \frac{1 + \nu}{2}} \left(3 [\tau^{\frac{1 +
     \nu}{2}} (- X^{\!\resizebox{0.6em}{!}{
\begin{tikzpicture}
\pgfpathmoveto{\pgfqpoint{0cm}{-0.035cm}}
\pgfpathlineto{\pgfqpoint{1.376cm}{-0.035cm}}
\pgfpathlineto{\pgfqpoint{1.376cm}{1.552cm}}
\pgfpathlineto{\pgfqpoint{0cm}{1.552cm}}
\pgfpathclose
\pgfusepath{clip}
\begin{pgfscope}
\begin{pgfscope}
\pgfpathmoveto{\pgfqpoint{0cm}{-0.035cm}}
\pgfpathlineto{\pgfqpoint{1.376cm}{-0.035cm}}
\pgfpathlineto{\pgfqpoint{1.376cm}{1.552cm}}
\pgfpathlineto{\pgfqpoint{0cm}{1.552cm}}
\pgfpathclose
\pgfusepath{clip}
\begin{pgfscope}
\begin{pgfscope}
\pgfsetdash{}{0cm}
\pgfsetlinewidth{0.818mm}
\pgfsetroundcap
\pgfsetroundjoin
\pgfsetmiterlimit{7.0}
\definecolor{eps2pgf_color}{gray}{0}\pgfsetstrokecolor{eps2pgf_color}\pgfsetfillcolor{eps2pgf_color}
\pgfpathmoveto{\pgfqpoint{0.117cm}{1.421cm}}
\pgfpathlineto{\pgfqpoint{0.682cm}{0.671cm}}
\pgfpathlineto{\pgfqpoint{1.246cm}{1.421cm}}
\pgfusepath{stroke}
\end{pgfscope}
\definecolor{eps2pgf_color}{gray}{0}\pgfsetstrokecolor{eps2pgf_color}\pgfsetfillcolor{eps2pgf_color}
\pgfpathmoveto{\pgfqpoint{0.273cm}{1.395cm}}
\pgfpathcurveto{\pgfqpoint{0.273cm}{1.432cm}}{\pgfqpoint{0.259cm}{1.467cm}}{\pgfqpoint{0.233cm}{1.492cm}}
\pgfpathcurveto{\pgfqpoint{0.207cm}{1.518cm}}{\pgfqpoint{0.173cm}{1.532cm}}{\pgfqpoint{0.137cm}{1.532cm}}
\pgfpathcurveto{\pgfqpoint{0.1cm}{1.532cm}}{\pgfqpoint{0.066cm}{1.518cm}}{\pgfqpoint{0.04cm}{1.492cm}}
\pgfpathcurveto{\pgfqpoint{0.014cm}{1.467cm}}{\pgfqpoint{0cm}{1.432cm}}{\pgfqpoint{0cm}{1.395cm}}
\pgfpathcurveto{\pgfqpoint{0cm}{1.359cm}}{\pgfqpoint{0.014cm}{1.324cm}}{\pgfqpoint{0.04cm}{1.299cm}}
\pgfpathcurveto{\pgfqpoint{0.066cm}{1.273cm}}{\pgfqpoint{0.1cm}{1.258cm}}{\pgfqpoint{0.137cm}{1.258cm}}
\pgfpathcurveto{\pgfqpoint{0.173cm}{1.258cm}}{\pgfqpoint{0.207cm}{1.273cm}}{\pgfqpoint{0.233cm}{1.299cm}}
\pgfpathcurveto{\pgfqpoint{0.259cm}{1.324cm}}{\pgfqpoint{0.273cm}{1.359cm}}{\pgfqpoint{0.273cm}{1.395cm}}
\pgfusepath{fill}
\begin{pgfscope}
\pgfsetdash{}{0cm}
\pgfsetlinewidth{0.818mm}
\pgfsetmiterlimit{7.0}
\pgfpathmoveto{\pgfqpoint{0.682cm}{0.671cm}}
\pgfpathlineto{\pgfqpoint{0.679cm}{1.418cm}}
\pgfusepath{stroke}
\end{pgfscope}
\pgfpathmoveto{\pgfqpoint{0.815cm}{1.399cm}}
\pgfpathcurveto{\pgfqpoint{0.815cm}{1.435cm}}{\pgfqpoint{0.801cm}{1.47cm}}{\pgfqpoint{0.775cm}{1.496cm}}
\pgfpathcurveto{\pgfqpoint{0.75cm}{1.521cm}}{\pgfqpoint{0.715cm}{1.536cm}}{\pgfqpoint{0.679cm}{1.536cm}}
\pgfpathcurveto{\pgfqpoint{0.643cm}{1.536cm}}{\pgfqpoint{0.608cm}{1.521cm}}{\pgfqpoint{0.582cm}{1.496cm}}
\pgfpathcurveto{\pgfqpoint{0.557cm}{1.47cm}}{\pgfqpoint{0.542cm}{1.435cm}}{\pgfqpoint{0.542cm}{1.399cm}}
\pgfpathcurveto{\pgfqpoint{0.542cm}{1.363cm}}{\pgfqpoint{0.557cm}{1.328cm}}{\pgfqpoint{0.582cm}{1.302cm}}
\pgfpathcurveto{\pgfqpoint{0.608cm}{1.276cm}}{\pgfqpoint{0.643cm}{1.262cm}}{\pgfqpoint{0.679cm}{1.262cm}}
\pgfpathcurveto{\pgfqpoint{0.715cm}{1.262cm}}{\pgfqpoint{0.75cm}{1.276cm}}{\pgfqpoint{0.775cm}{1.302cm}}
\pgfpathcurveto{\pgfqpoint{0.801cm}{1.328cm}}{\pgfqpoint{0.815cm}{1.363cm}}{\pgfqpoint{0.815cm}{1.399cm}}
\pgfusepath{fill}
\pgfpathmoveto{\pgfqpoint{1.345cm}{1.371cm}}
\pgfpathcurveto{\pgfqpoint{1.345cm}{1.408cm}}{\pgfqpoint{1.331cm}{1.442cm}}{\pgfqpoint{1.305cm}{1.468cm}}
\pgfpathcurveto{\pgfqpoint{1.28cm}{1.494cm}}{\pgfqpoint{1.245cm}{1.508cm}}{\pgfqpoint{1.209cm}{1.508cm}}
\pgfpathcurveto{\pgfqpoint{1.172cm}{1.508cm}}{\pgfqpoint{1.138cm}{1.494cm}}{\pgfqpoint{1.112cm}{1.468cm}}
\pgfpathcurveto{\pgfqpoint{1.087cm}{1.442cm}}{\pgfqpoint{1.072cm}{1.408cm}}{\pgfqpoint{1.072cm}{1.371cm}}
\pgfpathcurveto{\pgfqpoint{1.072cm}{1.335cm}}{\pgfqpoint{1.087cm}{1.3cm}}{\pgfqpoint{1.112cm}{1.274cm}}
\pgfpathcurveto{\pgfqpoint{1.138cm}{1.249cm}}{\pgfqpoint{1.172cm}{1.234cm}}{\pgfqpoint{1.209cm}{1.234cm}}
\pgfpathcurveto{\pgfqpoint{1.245cm}{1.234cm}}{\pgfqpoint{1.28cm}{1.249cm}}{\pgfqpoint{1.305cm}{1.274cm}}
\pgfpathcurveto{\pgfqpoint{1.331cm}{1.3cm}}{\pgfqpoint{1.345cm}{1.335cm}}{\pgfqpoint{1.345cm}{1.371cm}}
\pgfusepath{fill}
\begin{pgfscope}
\pgfsetdash{}{0cm}
\pgfsetlinewidth{0.818mm}
\pgfsetroundcap
\pgfsetmiterlimit{4.0}
\pgfpathmoveto{\pgfqpoint{0.682cm}{0.671cm}}
\pgfpathlineto{\pgfqpoint{0.682cm}{0.042cm}}
\pgfusepath{stroke}
\end{pgfscope}
\end{pgfscope}
\end{pgfscope}
\end{pgfscope}
\end{tikzpicture}}} + \phi + \psi)] \precprec X^{\!\resizebox{0.6em}{!}{
\begin{tikzpicture}
\pgfpathmoveto{\pgfqpoint{0cm}{0cm}}
\pgfpathlineto{\pgfqpoint{1.376cm}{0cm}}
\pgfpathlineto{\pgfqpoint{1.376cm}{1.588cm}}
\pgfpathlineto{\pgfqpoint{0cm}{1.588cm}}
\pgfpathclose
\pgfusepath{clip}
\begin{pgfscope}
\begin{pgfscope}
\pgfpathmoveto{\pgfqpoint{0cm}{0cm}}
\pgfpathlineto{\pgfqpoint{1.376cm}{0cm}}
\pgfpathlineto{\pgfqpoint{1.376cm}{1.588cm}}
\pgfpathlineto{\pgfqpoint{0cm}{1.588cm}}
\pgfpathclose
\pgfusepath{clip}
\begin{pgfscope}
\begin{pgfscope}
\definecolor{eps2pgf_color}{gray}{0.976471}\pgfsetstrokecolor{eps2pgf_color}\pgfsetfillcolor{eps2pgf_color}
\pgfpathmoveto{\pgfqpoint{0cm}{0cm}}
\pgfpathlineto{\pgfqpoint{1.376cm}{0cm}}
\pgfpathlineto{\pgfqpoint{1.376cm}{1.588cm}}
\pgfpathlineto{\pgfqpoint{0cm}{1.588cm}}
\pgfpathclose
\pgfusepath{fill}
\end{pgfscope}
\begin{pgfscope}
\pgfsetdash{}{0cm}
\pgfsetlinewidth{0.818mm}
\pgfsetroundcap
\pgfsetroundjoin
\pgfsetmiterlimit{7.0}
\definecolor{eps2pgf_color}{gray}{0}\pgfsetstrokecolor{eps2pgf_color}\pgfsetfillcolor{eps2pgf_color}
\pgfpathmoveto{\pgfqpoint{0.117cm}{1.476cm}}
\pgfpathlineto{\pgfqpoint{0.682cm}{0.726cm}}
\pgfpathlineto{\pgfqpoint{1.246cm}{1.476cm}}
\pgfusepath{stroke}
\end{pgfscope}
\definecolor{eps2pgf_color}{gray}{0}\pgfsetstrokecolor{eps2pgf_color}\pgfsetfillcolor{eps2pgf_color}
\pgfpathmoveto{\pgfqpoint{0.273cm}{1.451cm}}
\pgfpathcurveto{\pgfqpoint{0.273cm}{1.487cm}}{\pgfqpoint{0.259cm}{1.522cm}}{\pgfqpoint{0.233cm}{1.547cm}}
\pgfpathcurveto{\pgfqpoint{0.207cm}{1.573cm}}{\pgfqpoint{0.173cm}{1.588cm}}{\pgfqpoint{0.137cm}{1.588cm}}
\pgfpathcurveto{\pgfqpoint{0.1cm}{1.588cm}}{\pgfqpoint{0.066cm}{1.573cm}}{\pgfqpoint{0.04cm}{1.547cm}}
\pgfpathcurveto{\pgfqpoint{0.014cm}{1.522cm}}{\pgfqpoint{0cm}{1.487cm}}{\pgfqpoint{0cm}{1.451cm}}
\pgfpathcurveto{\pgfqpoint{0cm}{1.414cm}}{\pgfqpoint{0.014cm}{1.379cm}}{\pgfqpoint{0.04cm}{1.354cm}}
\pgfpathcurveto{\pgfqpoint{0.066cm}{1.328cm}}{\pgfqpoint{0.1cm}{1.314cm}}{\pgfqpoint{0.137cm}{1.314cm}}
\pgfpathcurveto{\pgfqpoint{0.173cm}{1.314cm}}{\pgfqpoint{0.207cm}{1.328cm}}{\pgfqpoint{0.233cm}{1.354cm}}
\pgfpathcurveto{\pgfqpoint{0.259cm}{1.379cm}}{\pgfqpoint{0.273cm}{1.414cm}}{\pgfqpoint{0.273cm}{1.451cm}}
\pgfusepath{fill}
\pgfpathmoveto{\pgfqpoint{1.345cm}{1.426cm}}
\pgfpathcurveto{\pgfqpoint{1.345cm}{1.463cm}}{\pgfqpoint{1.331cm}{1.497cm}}{\pgfqpoint{1.305cm}{1.523cm}}
\pgfpathcurveto{\pgfqpoint{1.28cm}{1.549cm}}{\pgfqpoint{1.245cm}{1.563cm}}{\pgfqpoint{1.209cm}{1.563cm}}
\pgfpathcurveto{\pgfqpoint{1.172cm}{1.563cm}}{\pgfqpoint{1.138cm}{1.549cm}}{\pgfqpoint{1.112cm}{1.523cm}}
\pgfpathcurveto{\pgfqpoint{1.087cm}{1.497cm}}{\pgfqpoint{1.072cm}{1.463cm}}{\pgfqpoint{1.072cm}{1.426cm}}
\pgfpathcurveto{\pgfqpoint{1.072cm}{1.39cm}}{\pgfqpoint{1.087cm}{1.355cm}}{\pgfqpoint{1.112cm}{1.329cm}}
\pgfpathcurveto{\pgfqpoint{1.138cm}{1.304cm}}{\pgfqpoint{1.172cm}{1.289cm}}{\pgfqpoint{1.209cm}{1.289cm}}
\pgfpathcurveto{\pgfqpoint{1.245cm}{1.289cm}}{\pgfqpoint{1.28cm}{1.304cm}}{\pgfqpoint{1.305cm}{1.329cm}}
\pgfpathcurveto{\pgfqpoint{1.331cm}{1.355cm}}{\pgfqpoint{1.345cm}{1.39cm}}{\pgfqpoint{1.345cm}{1.426cm}}
\pgfusepath{fill}
\begin{pgfscope}
\pgfsetdash{}{0cm}
\pgfsetlinewidth{0.818mm}
\pgfsetroundcap
\pgfsetmiterlimit{4.0}
\pgfpathmoveto{\pgfqpoint{0.682cm}{0.726cm}}
\pgfpathlineto{\pgfqpoint{0.682cm}{0.097cm}}
\pgfusepath{stroke}
\end{pgfscope}
\end{pgfscope}
\end{pgfscope}
\end{pgfscope}
\end{tikzpicture}}}\right)
  \end{equation}
  where $\nu > 0$ will be chosen later. This leads us to
  \[ 0 = \LL \breve{\vartheta} + \LL \psi - \left( \tau^{- \frac{1 + \nu}{2}}
     \left( 3 \left[ \tau^{\frac{1 + \nu}{2}} (- X^{\!\resizebox{0.6em}{!}{
\begin{tikzpicture}
\pgfpathmoveto{\pgfqpoint{0cm}{-0.035cm}}
\pgfpathlineto{\pgfqpoint{1.376cm}{-0.035cm}}
\pgfpathlineto{\pgfqpoint{1.376cm}{1.552cm}}
\pgfpathlineto{\pgfqpoint{0cm}{1.552cm}}
\pgfpathclose
\pgfusepath{clip}
\begin{pgfscope}
\begin{pgfscope}
\pgfpathmoveto{\pgfqpoint{0cm}{-0.035cm}}
\pgfpathlineto{\pgfqpoint{1.376cm}{-0.035cm}}
\pgfpathlineto{\pgfqpoint{1.376cm}{1.552cm}}
\pgfpathlineto{\pgfqpoint{0cm}{1.552cm}}
\pgfpathclose
\pgfusepath{clip}
\begin{pgfscope}
\begin{pgfscope}
\pgfsetdash{}{0cm}
\pgfsetlinewidth{0.818mm}
\pgfsetroundcap
\pgfsetroundjoin
\pgfsetmiterlimit{7.0}
\definecolor{eps2pgf_color}{gray}{0}\pgfsetstrokecolor{eps2pgf_color}\pgfsetfillcolor{eps2pgf_color}
\pgfpathmoveto{\pgfqpoint{0.117cm}{1.421cm}}
\pgfpathlineto{\pgfqpoint{0.682cm}{0.671cm}}
\pgfpathlineto{\pgfqpoint{1.246cm}{1.421cm}}
\pgfusepath{stroke}
\end{pgfscope}
\definecolor{eps2pgf_color}{gray}{0}\pgfsetstrokecolor{eps2pgf_color}\pgfsetfillcolor{eps2pgf_color}
\pgfpathmoveto{\pgfqpoint{0.273cm}{1.395cm}}
\pgfpathcurveto{\pgfqpoint{0.273cm}{1.432cm}}{\pgfqpoint{0.259cm}{1.467cm}}{\pgfqpoint{0.233cm}{1.492cm}}
\pgfpathcurveto{\pgfqpoint{0.207cm}{1.518cm}}{\pgfqpoint{0.173cm}{1.532cm}}{\pgfqpoint{0.137cm}{1.532cm}}
\pgfpathcurveto{\pgfqpoint{0.1cm}{1.532cm}}{\pgfqpoint{0.066cm}{1.518cm}}{\pgfqpoint{0.04cm}{1.492cm}}
\pgfpathcurveto{\pgfqpoint{0.014cm}{1.467cm}}{\pgfqpoint{0cm}{1.432cm}}{\pgfqpoint{0cm}{1.395cm}}
\pgfpathcurveto{\pgfqpoint{0cm}{1.359cm}}{\pgfqpoint{0.014cm}{1.324cm}}{\pgfqpoint{0.04cm}{1.299cm}}
\pgfpathcurveto{\pgfqpoint{0.066cm}{1.273cm}}{\pgfqpoint{0.1cm}{1.258cm}}{\pgfqpoint{0.137cm}{1.258cm}}
\pgfpathcurveto{\pgfqpoint{0.173cm}{1.258cm}}{\pgfqpoint{0.207cm}{1.273cm}}{\pgfqpoint{0.233cm}{1.299cm}}
\pgfpathcurveto{\pgfqpoint{0.259cm}{1.324cm}}{\pgfqpoint{0.273cm}{1.359cm}}{\pgfqpoint{0.273cm}{1.395cm}}
\pgfusepath{fill}
\begin{pgfscope}
\pgfsetdash{}{0cm}
\pgfsetlinewidth{0.818mm}
\pgfsetmiterlimit{7.0}
\pgfpathmoveto{\pgfqpoint{0.682cm}{0.671cm}}
\pgfpathlineto{\pgfqpoint{0.679cm}{1.418cm}}
\pgfusepath{stroke}
\end{pgfscope}
\pgfpathmoveto{\pgfqpoint{0.815cm}{1.399cm}}
\pgfpathcurveto{\pgfqpoint{0.815cm}{1.435cm}}{\pgfqpoint{0.801cm}{1.47cm}}{\pgfqpoint{0.775cm}{1.496cm}}
\pgfpathcurveto{\pgfqpoint{0.75cm}{1.521cm}}{\pgfqpoint{0.715cm}{1.536cm}}{\pgfqpoint{0.679cm}{1.536cm}}
\pgfpathcurveto{\pgfqpoint{0.643cm}{1.536cm}}{\pgfqpoint{0.608cm}{1.521cm}}{\pgfqpoint{0.582cm}{1.496cm}}
\pgfpathcurveto{\pgfqpoint{0.557cm}{1.47cm}}{\pgfqpoint{0.542cm}{1.435cm}}{\pgfqpoint{0.542cm}{1.399cm}}
\pgfpathcurveto{\pgfqpoint{0.542cm}{1.363cm}}{\pgfqpoint{0.557cm}{1.328cm}}{\pgfqpoint{0.582cm}{1.302cm}}
\pgfpathcurveto{\pgfqpoint{0.608cm}{1.276cm}}{\pgfqpoint{0.643cm}{1.262cm}}{\pgfqpoint{0.679cm}{1.262cm}}
\pgfpathcurveto{\pgfqpoint{0.715cm}{1.262cm}}{\pgfqpoint{0.75cm}{1.276cm}}{\pgfqpoint{0.775cm}{1.302cm}}
\pgfpathcurveto{\pgfqpoint{0.801cm}{1.328cm}}{\pgfqpoint{0.815cm}{1.363cm}}{\pgfqpoint{0.815cm}{1.399cm}}
\pgfusepath{fill}
\pgfpathmoveto{\pgfqpoint{1.345cm}{1.371cm}}
\pgfpathcurveto{\pgfqpoint{1.345cm}{1.408cm}}{\pgfqpoint{1.331cm}{1.442cm}}{\pgfqpoint{1.305cm}{1.468cm}}
\pgfpathcurveto{\pgfqpoint{1.28cm}{1.494cm}}{\pgfqpoint{1.245cm}{1.508cm}}{\pgfqpoint{1.209cm}{1.508cm}}
\pgfpathcurveto{\pgfqpoint{1.172cm}{1.508cm}}{\pgfqpoint{1.138cm}{1.494cm}}{\pgfqpoint{1.112cm}{1.468cm}}
\pgfpathcurveto{\pgfqpoint{1.087cm}{1.442cm}}{\pgfqpoint{1.072cm}{1.408cm}}{\pgfqpoint{1.072cm}{1.371cm}}
\pgfpathcurveto{\pgfqpoint{1.072cm}{1.335cm}}{\pgfqpoint{1.087cm}{1.3cm}}{\pgfqpoint{1.112cm}{1.274cm}}
\pgfpathcurveto{\pgfqpoint{1.138cm}{1.249cm}}{\pgfqpoint{1.172cm}{1.234cm}}{\pgfqpoint{1.209cm}{1.234cm}}
\pgfpathcurveto{\pgfqpoint{1.245cm}{1.234cm}}{\pgfqpoint{1.28cm}{1.249cm}}{\pgfqpoint{1.305cm}{1.274cm}}
\pgfpathcurveto{\pgfqpoint{1.331cm}{1.3cm}}{\pgfqpoint{1.345cm}{1.335cm}}{\pgfqpoint{1.345cm}{1.371cm}}
\pgfusepath{fill}
\begin{pgfscope}
\pgfsetdash{}{0cm}
\pgfsetlinewidth{0.818mm}
\pgfsetroundcap
\pgfsetmiterlimit{4.0}
\pgfpathmoveto{\pgfqpoint{0.682cm}{0.671cm}}
\pgfpathlineto{\pgfqpoint{0.682cm}{0.042cm}}
\pgfusepath{stroke}
\end{pgfscope}
\end{pgfscope}
\end{pgfscope}
\end{pgfscope}
\end{tikzpicture}}} + \phi + \psi)
     \right] \precprec \llbracket X^2\rrbracket \right) - 3 (- X^{\!\resizebox{0.6em}{!}{
\begin{tikzpicture}
\pgfpathmoveto{\pgfqpoint{0cm}{-0.035cm}}
\pgfpathlineto{\pgfqpoint{1.376cm}{-0.035cm}}
\pgfpathlineto{\pgfqpoint{1.376cm}{1.552cm}}
\pgfpathlineto{\pgfqpoint{0cm}{1.552cm}}
\pgfpathclose
\pgfusepath{clip}
\begin{pgfscope}
\begin{pgfscope}
\pgfpathmoveto{\pgfqpoint{0cm}{-0.035cm}}
\pgfpathlineto{\pgfqpoint{1.376cm}{-0.035cm}}
\pgfpathlineto{\pgfqpoint{1.376cm}{1.552cm}}
\pgfpathlineto{\pgfqpoint{0cm}{1.552cm}}
\pgfpathclose
\pgfusepath{clip}
\begin{pgfscope}
\begin{pgfscope}
\pgfsetdash{}{0cm}
\pgfsetlinewidth{0.818mm}
\pgfsetroundcap
\pgfsetroundjoin
\pgfsetmiterlimit{7.0}
\definecolor{eps2pgf_color}{gray}{0}\pgfsetstrokecolor{eps2pgf_color}\pgfsetfillcolor{eps2pgf_color}
\pgfpathmoveto{\pgfqpoint{0.117cm}{1.421cm}}
\pgfpathlineto{\pgfqpoint{0.682cm}{0.671cm}}
\pgfpathlineto{\pgfqpoint{1.246cm}{1.421cm}}
\pgfusepath{stroke}
\end{pgfscope}
\definecolor{eps2pgf_color}{gray}{0}\pgfsetstrokecolor{eps2pgf_color}\pgfsetfillcolor{eps2pgf_color}
\pgfpathmoveto{\pgfqpoint{0.273cm}{1.395cm}}
\pgfpathcurveto{\pgfqpoint{0.273cm}{1.432cm}}{\pgfqpoint{0.259cm}{1.467cm}}{\pgfqpoint{0.233cm}{1.492cm}}
\pgfpathcurveto{\pgfqpoint{0.207cm}{1.518cm}}{\pgfqpoint{0.173cm}{1.532cm}}{\pgfqpoint{0.137cm}{1.532cm}}
\pgfpathcurveto{\pgfqpoint{0.1cm}{1.532cm}}{\pgfqpoint{0.066cm}{1.518cm}}{\pgfqpoint{0.04cm}{1.492cm}}
\pgfpathcurveto{\pgfqpoint{0.014cm}{1.467cm}}{\pgfqpoint{0cm}{1.432cm}}{\pgfqpoint{0cm}{1.395cm}}
\pgfpathcurveto{\pgfqpoint{0cm}{1.359cm}}{\pgfqpoint{0.014cm}{1.324cm}}{\pgfqpoint{0.04cm}{1.299cm}}
\pgfpathcurveto{\pgfqpoint{0.066cm}{1.273cm}}{\pgfqpoint{0.1cm}{1.258cm}}{\pgfqpoint{0.137cm}{1.258cm}}
\pgfpathcurveto{\pgfqpoint{0.173cm}{1.258cm}}{\pgfqpoint{0.207cm}{1.273cm}}{\pgfqpoint{0.233cm}{1.299cm}}
\pgfpathcurveto{\pgfqpoint{0.259cm}{1.324cm}}{\pgfqpoint{0.273cm}{1.359cm}}{\pgfqpoint{0.273cm}{1.395cm}}
\pgfusepath{fill}
\begin{pgfscope}
\pgfsetdash{}{0cm}
\pgfsetlinewidth{0.818mm}
\pgfsetmiterlimit{7.0}
\pgfpathmoveto{\pgfqpoint{0.682cm}{0.671cm}}
\pgfpathlineto{\pgfqpoint{0.679cm}{1.418cm}}
\pgfusepath{stroke}
\end{pgfscope}
\pgfpathmoveto{\pgfqpoint{0.815cm}{1.399cm}}
\pgfpathcurveto{\pgfqpoint{0.815cm}{1.435cm}}{\pgfqpoint{0.801cm}{1.47cm}}{\pgfqpoint{0.775cm}{1.496cm}}
\pgfpathcurveto{\pgfqpoint{0.75cm}{1.521cm}}{\pgfqpoint{0.715cm}{1.536cm}}{\pgfqpoint{0.679cm}{1.536cm}}
\pgfpathcurveto{\pgfqpoint{0.643cm}{1.536cm}}{\pgfqpoint{0.608cm}{1.521cm}}{\pgfqpoint{0.582cm}{1.496cm}}
\pgfpathcurveto{\pgfqpoint{0.557cm}{1.47cm}}{\pgfqpoint{0.542cm}{1.435cm}}{\pgfqpoint{0.542cm}{1.399cm}}
\pgfpathcurveto{\pgfqpoint{0.542cm}{1.363cm}}{\pgfqpoint{0.557cm}{1.328cm}}{\pgfqpoint{0.582cm}{1.302cm}}
\pgfpathcurveto{\pgfqpoint{0.608cm}{1.276cm}}{\pgfqpoint{0.643cm}{1.262cm}}{\pgfqpoint{0.679cm}{1.262cm}}
\pgfpathcurveto{\pgfqpoint{0.715cm}{1.262cm}}{\pgfqpoint{0.75cm}{1.276cm}}{\pgfqpoint{0.775cm}{1.302cm}}
\pgfpathcurveto{\pgfqpoint{0.801cm}{1.328cm}}{\pgfqpoint{0.815cm}{1.363cm}}{\pgfqpoint{0.815cm}{1.399cm}}
\pgfusepath{fill}
\pgfpathmoveto{\pgfqpoint{1.345cm}{1.371cm}}
\pgfpathcurveto{\pgfqpoint{1.345cm}{1.408cm}}{\pgfqpoint{1.331cm}{1.442cm}}{\pgfqpoint{1.305cm}{1.468cm}}
\pgfpathcurveto{\pgfqpoint{1.28cm}{1.494cm}}{\pgfqpoint{1.245cm}{1.508cm}}{\pgfqpoint{1.209cm}{1.508cm}}
\pgfpathcurveto{\pgfqpoint{1.172cm}{1.508cm}}{\pgfqpoint{1.138cm}{1.494cm}}{\pgfqpoint{1.112cm}{1.468cm}}
\pgfpathcurveto{\pgfqpoint{1.087cm}{1.442cm}}{\pgfqpoint{1.072cm}{1.408cm}}{\pgfqpoint{1.072cm}{1.371cm}}
\pgfpathcurveto{\pgfqpoint{1.072cm}{1.335cm}}{\pgfqpoint{1.087cm}{1.3cm}}{\pgfqpoint{1.112cm}{1.274cm}}
\pgfpathcurveto{\pgfqpoint{1.138cm}{1.249cm}}{\pgfqpoint{1.172cm}{1.234cm}}{\pgfqpoint{1.209cm}{1.234cm}}
\pgfpathcurveto{\pgfqpoint{1.245cm}{1.234cm}}{\pgfqpoint{1.28cm}{1.249cm}}{\pgfqpoint{1.305cm}{1.274cm}}
\pgfpathcurveto{\pgfqpoint{1.331cm}{1.3cm}}{\pgfqpoint{1.345cm}{1.335cm}}{\pgfqpoint{1.345cm}{1.371cm}}
\pgfusepath{fill}
\begin{pgfscope}
\pgfsetdash{}{0cm}
\pgfsetlinewidth{0.818mm}
\pgfsetroundcap
\pgfsetmiterlimit{4.0}
\pgfpathmoveto{\pgfqpoint{0.682cm}{0.671cm}}
\pgfpathlineto{\pgfqpoint{0.682cm}{0.042cm}}
\pgfusepath{stroke}
\end{pgfscope}
\end{pgfscope}
\end{pgfscope}
\end{pgfscope}
\end{tikzpicture}}} + \phi + \psi)
     \prec \llbracket X^2\rrbracket \right) \]
  \[ + 3 \llbracket X^2\rrbracket \preccurlyeq (- X^{\!\resizebox{0.6em}{!}{
\begin{tikzpicture}
\pgfpathmoveto{\pgfqpoint{0cm}{-0.035cm}}
\pgfpathlineto{\pgfqpoint{1.376cm}{-0.035cm}}
\pgfpathlineto{\pgfqpoint{1.376cm}{1.552cm}}
\pgfpathlineto{\pgfqpoint{0cm}{1.552cm}}
\pgfpathclose
\pgfusepath{clip}
\begin{pgfscope}
\begin{pgfscope}
\pgfpathmoveto{\pgfqpoint{0cm}{-0.035cm}}
\pgfpathlineto{\pgfqpoint{1.376cm}{-0.035cm}}
\pgfpathlineto{\pgfqpoint{1.376cm}{1.552cm}}
\pgfpathlineto{\pgfqpoint{0cm}{1.552cm}}
\pgfpathclose
\pgfusepath{clip}
\begin{pgfscope}
\begin{pgfscope}
\pgfsetdash{}{0cm}
\pgfsetlinewidth{0.818mm}
\pgfsetroundcap
\pgfsetroundjoin
\pgfsetmiterlimit{7.0}
\definecolor{eps2pgf_color}{gray}{0}\pgfsetstrokecolor{eps2pgf_color}\pgfsetfillcolor{eps2pgf_color}
\pgfpathmoveto{\pgfqpoint{0.117cm}{1.421cm}}
\pgfpathlineto{\pgfqpoint{0.682cm}{0.671cm}}
\pgfpathlineto{\pgfqpoint{1.246cm}{1.421cm}}
\pgfusepath{stroke}
\end{pgfscope}
\definecolor{eps2pgf_color}{gray}{0}\pgfsetstrokecolor{eps2pgf_color}\pgfsetfillcolor{eps2pgf_color}
\pgfpathmoveto{\pgfqpoint{0.273cm}{1.395cm}}
\pgfpathcurveto{\pgfqpoint{0.273cm}{1.432cm}}{\pgfqpoint{0.259cm}{1.467cm}}{\pgfqpoint{0.233cm}{1.492cm}}
\pgfpathcurveto{\pgfqpoint{0.207cm}{1.518cm}}{\pgfqpoint{0.173cm}{1.532cm}}{\pgfqpoint{0.137cm}{1.532cm}}
\pgfpathcurveto{\pgfqpoint{0.1cm}{1.532cm}}{\pgfqpoint{0.066cm}{1.518cm}}{\pgfqpoint{0.04cm}{1.492cm}}
\pgfpathcurveto{\pgfqpoint{0.014cm}{1.467cm}}{\pgfqpoint{0cm}{1.432cm}}{\pgfqpoint{0cm}{1.395cm}}
\pgfpathcurveto{\pgfqpoint{0cm}{1.359cm}}{\pgfqpoint{0.014cm}{1.324cm}}{\pgfqpoint{0.04cm}{1.299cm}}
\pgfpathcurveto{\pgfqpoint{0.066cm}{1.273cm}}{\pgfqpoint{0.1cm}{1.258cm}}{\pgfqpoint{0.137cm}{1.258cm}}
\pgfpathcurveto{\pgfqpoint{0.173cm}{1.258cm}}{\pgfqpoint{0.207cm}{1.273cm}}{\pgfqpoint{0.233cm}{1.299cm}}
\pgfpathcurveto{\pgfqpoint{0.259cm}{1.324cm}}{\pgfqpoint{0.273cm}{1.359cm}}{\pgfqpoint{0.273cm}{1.395cm}}
\pgfusepath{fill}
\begin{pgfscope}
\pgfsetdash{}{0cm}
\pgfsetlinewidth{0.818mm}
\pgfsetmiterlimit{7.0}
\pgfpathmoveto{\pgfqpoint{0.682cm}{0.671cm}}
\pgfpathlineto{\pgfqpoint{0.679cm}{1.418cm}}
\pgfusepath{stroke}
\end{pgfscope}
\pgfpathmoveto{\pgfqpoint{0.815cm}{1.399cm}}
\pgfpathcurveto{\pgfqpoint{0.815cm}{1.435cm}}{\pgfqpoint{0.801cm}{1.47cm}}{\pgfqpoint{0.775cm}{1.496cm}}
\pgfpathcurveto{\pgfqpoint{0.75cm}{1.521cm}}{\pgfqpoint{0.715cm}{1.536cm}}{\pgfqpoint{0.679cm}{1.536cm}}
\pgfpathcurveto{\pgfqpoint{0.643cm}{1.536cm}}{\pgfqpoint{0.608cm}{1.521cm}}{\pgfqpoint{0.582cm}{1.496cm}}
\pgfpathcurveto{\pgfqpoint{0.557cm}{1.47cm}}{\pgfqpoint{0.542cm}{1.435cm}}{\pgfqpoint{0.542cm}{1.399cm}}
\pgfpathcurveto{\pgfqpoint{0.542cm}{1.363cm}}{\pgfqpoint{0.557cm}{1.328cm}}{\pgfqpoint{0.582cm}{1.302cm}}
\pgfpathcurveto{\pgfqpoint{0.608cm}{1.276cm}}{\pgfqpoint{0.643cm}{1.262cm}}{\pgfqpoint{0.679cm}{1.262cm}}
\pgfpathcurveto{\pgfqpoint{0.715cm}{1.262cm}}{\pgfqpoint{0.75cm}{1.276cm}}{\pgfqpoint{0.775cm}{1.302cm}}
\pgfpathcurveto{\pgfqpoint{0.801cm}{1.328cm}}{\pgfqpoint{0.815cm}{1.363cm}}{\pgfqpoint{0.815cm}{1.399cm}}
\pgfusepath{fill}
\pgfpathmoveto{\pgfqpoint{1.345cm}{1.371cm}}
\pgfpathcurveto{\pgfqpoint{1.345cm}{1.408cm}}{\pgfqpoint{1.331cm}{1.442cm}}{\pgfqpoint{1.305cm}{1.468cm}}
\pgfpathcurveto{\pgfqpoint{1.28cm}{1.494cm}}{\pgfqpoint{1.245cm}{1.508cm}}{\pgfqpoint{1.209cm}{1.508cm}}
\pgfpathcurveto{\pgfqpoint{1.172cm}{1.508cm}}{\pgfqpoint{1.138cm}{1.494cm}}{\pgfqpoint{1.112cm}{1.468cm}}
\pgfpathcurveto{\pgfqpoint{1.087cm}{1.442cm}}{\pgfqpoint{1.072cm}{1.408cm}}{\pgfqpoint{1.072cm}{1.371cm}}
\pgfpathcurveto{\pgfqpoint{1.072cm}{1.335cm}}{\pgfqpoint{1.087cm}{1.3cm}}{\pgfqpoint{1.112cm}{1.274cm}}
\pgfpathcurveto{\pgfqpoint{1.138cm}{1.249cm}}{\pgfqpoint{1.172cm}{1.234cm}}{\pgfqpoint{1.209cm}{1.234cm}}
\pgfpathcurveto{\pgfqpoint{1.245cm}{1.234cm}}{\pgfqpoint{1.28cm}{1.249cm}}{\pgfqpoint{1.305cm}{1.274cm}}
\pgfpathcurveto{\pgfqpoint{1.331cm}{1.3cm}}{\pgfqpoint{1.345cm}{1.335cm}}{\pgfqpoint{1.345cm}{1.371cm}}
\pgfusepath{fill}
\begin{pgfscope}
\pgfsetdash{}{0cm}
\pgfsetlinewidth{0.818mm}
\pgfsetroundcap
\pgfsetmiterlimit{4.0}
\pgfpathmoveto{\pgfqpoint{0.682cm}{0.671cm}}
\pgfpathlineto{\pgfqpoint{0.682cm}{0.042cm}}
\pgfusepath{stroke}
\end{pgfscope}
\end{pgfscope}
\end{pgfscope}
\end{pgfscope}
\end{tikzpicture}}} + \phi + \psi) \]
  \[ - \left( \LL \left[ \tau^{- \frac{1 + \nu}{2}} \left( \left[ 3
     \tau^{\frac{1 + \nu}{2}} (- X^{\!\resizebox{0.6em}{!}{
\begin{tikzpicture}
\pgfpathmoveto{\pgfqpoint{0cm}{-0.035cm}}
\pgfpathlineto{\pgfqpoint{1.376cm}{-0.035cm}}
\pgfpathlineto{\pgfqpoint{1.376cm}{1.552cm}}
\pgfpathlineto{\pgfqpoint{0cm}{1.552cm}}
\pgfpathclose
\pgfusepath{clip}
\begin{pgfscope}
\begin{pgfscope}
\pgfpathmoveto{\pgfqpoint{0cm}{-0.035cm}}
\pgfpathlineto{\pgfqpoint{1.376cm}{-0.035cm}}
\pgfpathlineto{\pgfqpoint{1.376cm}{1.552cm}}
\pgfpathlineto{\pgfqpoint{0cm}{1.552cm}}
\pgfpathclose
\pgfusepath{clip}
\begin{pgfscope}
\begin{pgfscope}
\pgfsetdash{}{0cm}
\pgfsetlinewidth{0.818mm}
\pgfsetroundcap
\pgfsetroundjoin
\pgfsetmiterlimit{7.0}
\definecolor{eps2pgf_color}{gray}{0}\pgfsetstrokecolor{eps2pgf_color}\pgfsetfillcolor{eps2pgf_color}
\pgfpathmoveto{\pgfqpoint{0.117cm}{1.421cm}}
\pgfpathlineto{\pgfqpoint{0.682cm}{0.671cm}}
\pgfpathlineto{\pgfqpoint{1.246cm}{1.421cm}}
\pgfusepath{stroke}
\end{pgfscope}
\definecolor{eps2pgf_color}{gray}{0}\pgfsetstrokecolor{eps2pgf_color}\pgfsetfillcolor{eps2pgf_color}
\pgfpathmoveto{\pgfqpoint{0.273cm}{1.395cm}}
\pgfpathcurveto{\pgfqpoint{0.273cm}{1.432cm}}{\pgfqpoint{0.259cm}{1.467cm}}{\pgfqpoint{0.233cm}{1.492cm}}
\pgfpathcurveto{\pgfqpoint{0.207cm}{1.518cm}}{\pgfqpoint{0.173cm}{1.532cm}}{\pgfqpoint{0.137cm}{1.532cm}}
\pgfpathcurveto{\pgfqpoint{0.1cm}{1.532cm}}{\pgfqpoint{0.066cm}{1.518cm}}{\pgfqpoint{0.04cm}{1.492cm}}
\pgfpathcurveto{\pgfqpoint{0.014cm}{1.467cm}}{\pgfqpoint{0cm}{1.432cm}}{\pgfqpoint{0cm}{1.395cm}}
\pgfpathcurveto{\pgfqpoint{0cm}{1.359cm}}{\pgfqpoint{0.014cm}{1.324cm}}{\pgfqpoint{0.04cm}{1.299cm}}
\pgfpathcurveto{\pgfqpoint{0.066cm}{1.273cm}}{\pgfqpoint{0.1cm}{1.258cm}}{\pgfqpoint{0.137cm}{1.258cm}}
\pgfpathcurveto{\pgfqpoint{0.173cm}{1.258cm}}{\pgfqpoint{0.207cm}{1.273cm}}{\pgfqpoint{0.233cm}{1.299cm}}
\pgfpathcurveto{\pgfqpoint{0.259cm}{1.324cm}}{\pgfqpoint{0.273cm}{1.359cm}}{\pgfqpoint{0.273cm}{1.395cm}}
\pgfusepath{fill}
\begin{pgfscope}
\pgfsetdash{}{0cm}
\pgfsetlinewidth{0.818mm}
\pgfsetmiterlimit{7.0}
\pgfpathmoveto{\pgfqpoint{0.682cm}{0.671cm}}
\pgfpathlineto{\pgfqpoint{0.679cm}{1.418cm}}
\pgfusepath{stroke}
\end{pgfscope}
\pgfpathmoveto{\pgfqpoint{0.815cm}{1.399cm}}
\pgfpathcurveto{\pgfqpoint{0.815cm}{1.435cm}}{\pgfqpoint{0.801cm}{1.47cm}}{\pgfqpoint{0.775cm}{1.496cm}}
\pgfpathcurveto{\pgfqpoint{0.75cm}{1.521cm}}{\pgfqpoint{0.715cm}{1.536cm}}{\pgfqpoint{0.679cm}{1.536cm}}
\pgfpathcurveto{\pgfqpoint{0.643cm}{1.536cm}}{\pgfqpoint{0.608cm}{1.521cm}}{\pgfqpoint{0.582cm}{1.496cm}}
\pgfpathcurveto{\pgfqpoint{0.557cm}{1.47cm}}{\pgfqpoint{0.542cm}{1.435cm}}{\pgfqpoint{0.542cm}{1.399cm}}
\pgfpathcurveto{\pgfqpoint{0.542cm}{1.363cm}}{\pgfqpoint{0.557cm}{1.328cm}}{\pgfqpoint{0.582cm}{1.302cm}}
\pgfpathcurveto{\pgfqpoint{0.608cm}{1.276cm}}{\pgfqpoint{0.643cm}{1.262cm}}{\pgfqpoint{0.679cm}{1.262cm}}
\pgfpathcurveto{\pgfqpoint{0.715cm}{1.262cm}}{\pgfqpoint{0.75cm}{1.276cm}}{\pgfqpoint{0.775cm}{1.302cm}}
\pgfpathcurveto{\pgfqpoint{0.801cm}{1.328cm}}{\pgfqpoint{0.815cm}{1.363cm}}{\pgfqpoint{0.815cm}{1.399cm}}
\pgfusepath{fill}
\pgfpathmoveto{\pgfqpoint{1.345cm}{1.371cm}}
\pgfpathcurveto{\pgfqpoint{1.345cm}{1.408cm}}{\pgfqpoint{1.331cm}{1.442cm}}{\pgfqpoint{1.305cm}{1.468cm}}
\pgfpathcurveto{\pgfqpoint{1.28cm}{1.494cm}}{\pgfqpoint{1.245cm}{1.508cm}}{\pgfqpoint{1.209cm}{1.508cm}}
\pgfpathcurveto{\pgfqpoint{1.172cm}{1.508cm}}{\pgfqpoint{1.138cm}{1.494cm}}{\pgfqpoint{1.112cm}{1.468cm}}
\pgfpathcurveto{\pgfqpoint{1.087cm}{1.442cm}}{\pgfqpoint{1.072cm}{1.408cm}}{\pgfqpoint{1.072cm}{1.371cm}}
\pgfpathcurveto{\pgfqpoint{1.072cm}{1.335cm}}{\pgfqpoint{1.087cm}{1.3cm}}{\pgfqpoint{1.112cm}{1.274cm}}
\pgfpathcurveto{\pgfqpoint{1.138cm}{1.249cm}}{\pgfqpoint{1.172cm}{1.234cm}}{\pgfqpoint{1.209cm}{1.234cm}}
\pgfpathcurveto{\pgfqpoint{1.245cm}{1.234cm}}{\pgfqpoint{1.28cm}{1.249cm}}{\pgfqpoint{1.305cm}{1.274cm}}
\pgfpathcurveto{\pgfqpoint{1.331cm}{1.3cm}}{\pgfqpoint{1.345cm}{1.335cm}}{\pgfqpoint{1.345cm}{1.371cm}}
\pgfusepath{fill}
\begin{pgfscope}
\pgfsetdash{}{0cm}
\pgfsetlinewidth{0.818mm}
\pgfsetroundcap
\pgfsetmiterlimit{4.0}
\pgfpathmoveto{\pgfqpoint{0.682cm}{0.671cm}}
\pgfpathlineto{\pgfqpoint{0.682cm}{0.042cm}}
\pgfusepath{stroke}
\end{pgfscope}
\end{pgfscope}
\end{pgfscope}
\end{pgfscope}
\end{tikzpicture}}} + \phi + \psi) \right]
     \precprec X^{\!\resizebox{0.6em}{!}{
\begin{tikzpicture}
\pgfpathmoveto{\pgfqpoint{0cm}{0cm}}
\pgfpathlineto{\pgfqpoint{1.376cm}{0cm}}
\pgfpathlineto{\pgfqpoint{1.376cm}{1.588cm}}
\pgfpathlineto{\pgfqpoint{0cm}{1.588cm}}
\pgfpathclose
\pgfusepath{clip}
\begin{pgfscope}
\begin{pgfscope}
\pgfpathmoveto{\pgfqpoint{0cm}{0cm}}
\pgfpathlineto{\pgfqpoint{1.376cm}{0cm}}
\pgfpathlineto{\pgfqpoint{1.376cm}{1.588cm}}
\pgfpathlineto{\pgfqpoint{0cm}{1.588cm}}
\pgfpathclose
\pgfusepath{clip}
\begin{pgfscope}
\begin{pgfscope}
\definecolor{eps2pgf_color}{gray}{0.976471}\pgfsetstrokecolor{eps2pgf_color}\pgfsetfillcolor{eps2pgf_color}
\pgfpathmoveto{\pgfqpoint{0cm}{0cm}}
\pgfpathlineto{\pgfqpoint{1.376cm}{0cm}}
\pgfpathlineto{\pgfqpoint{1.376cm}{1.588cm}}
\pgfpathlineto{\pgfqpoint{0cm}{1.588cm}}
\pgfpathclose
\pgfusepath{fill}
\end{pgfscope}
\begin{pgfscope}
\pgfsetdash{}{0cm}
\pgfsetlinewidth{0.818mm}
\pgfsetroundcap
\pgfsetroundjoin
\pgfsetmiterlimit{7.0}
\definecolor{eps2pgf_color}{gray}{0}\pgfsetstrokecolor{eps2pgf_color}\pgfsetfillcolor{eps2pgf_color}
\pgfpathmoveto{\pgfqpoint{0.117cm}{1.476cm}}
\pgfpathlineto{\pgfqpoint{0.682cm}{0.726cm}}
\pgfpathlineto{\pgfqpoint{1.246cm}{1.476cm}}
\pgfusepath{stroke}
\end{pgfscope}
\definecolor{eps2pgf_color}{gray}{0}\pgfsetstrokecolor{eps2pgf_color}\pgfsetfillcolor{eps2pgf_color}
\pgfpathmoveto{\pgfqpoint{0.273cm}{1.451cm}}
\pgfpathcurveto{\pgfqpoint{0.273cm}{1.487cm}}{\pgfqpoint{0.259cm}{1.522cm}}{\pgfqpoint{0.233cm}{1.547cm}}
\pgfpathcurveto{\pgfqpoint{0.207cm}{1.573cm}}{\pgfqpoint{0.173cm}{1.588cm}}{\pgfqpoint{0.137cm}{1.588cm}}
\pgfpathcurveto{\pgfqpoint{0.1cm}{1.588cm}}{\pgfqpoint{0.066cm}{1.573cm}}{\pgfqpoint{0.04cm}{1.547cm}}
\pgfpathcurveto{\pgfqpoint{0.014cm}{1.522cm}}{\pgfqpoint{0cm}{1.487cm}}{\pgfqpoint{0cm}{1.451cm}}
\pgfpathcurveto{\pgfqpoint{0cm}{1.414cm}}{\pgfqpoint{0.014cm}{1.379cm}}{\pgfqpoint{0.04cm}{1.354cm}}
\pgfpathcurveto{\pgfqpoint{0.066cm}{1.328cm}}{\pgfqpoint{0.1cm}{1.314cm}}{\pgfqpoint{0.137cm}{1.314cm}}
\pgfpathcurveto{\pgfqpoint{0.173cm}{1.314cm}}{\pgfqpoint{0.207cm}{1.328cm}}{\pgfqpoint{0.233cm}{1.354cm}}
\pgfpathcurveto{\pgfqpoint{0.259cm}{1.379cm}}{\pgfqpoint{0.273cm}{1.414cm}}{\pgfqpoint{0.273cm}{1.451cm}}
\pgfusepath{fill}
\pgfpathmoveto{\pgfqpoint{1.345cm}{1.426cm}}
\pgfpathcurveto{\pgfqpoint{1.345cm}{1.463cm}}{\pgfqpoint{1.331cm}{1.497cm}}{\pgfqpoint{1.305cm}{1.523cm}}
\pgfpathcurveto{\pgfqpoint{1.28cm}{1.549cm}}{\pgfqpoint{1.245cm}{1.563cm}}{\pgfqpoint{1.209cm}{1.563cm}}
\pgfpathcurveto{\pgfqpoint{1.172cm}{1.563cm}}{\pgfqpoint{1.138cm}{1.549cm}}{\pgfqpoint{1.112cm}{1.523cm}}
\pgfpathcurveto{\pgfqpoint{1.087cm}{1.497cm}}{\pgfqpoint{1.072cm}{1.463cm}}{\pgfqpoint{1.072cm}{1.426cm}}
\pgfpathcurveto{\pgfqpoint{1.072cm}{1.39cm}}{\pgfqpoint{1.087cm}{1.355cm}}{\pgfqpoint{1.112cm}{1.329cm}}
\pgfpathcurveto{\pgfqpoint{1.138cm}{1.304cm}}{\pgfqpoint{1.172cm}{1.289cm}}{\pgfqpoint{1.209cm}{1.289cm}}
\pgfpathcurveto{\pgfqpoint{1.245cm}{1.289cm}}{\pgfqpoint{1.28cm}{1.304cm}}{\pgfqpoint{1.305cm}{1.329cm}}
\pgfpathcurveto{\pgfqpoint{1.331cm}{1.355cm}}{\pgfqpoint{1.345cm}{1.39cm}}{\pgfqpoint{1.345cm}{1.426cm}}
\pgfusepath{fill}
\begin{pgfscope}
\pgfsetdash{}{0cm}
\pgfsetlinewidth{0.818mm}
\pgfsetroundcap
\pgfsetmiterlimit{4.0}
\pgfpathmoveto{\pgfqpoint{0.682cm}{0.726cm}}
\pgfpathlineto{\pgfqpoint{0.682cm}{0.097cm}}
\pgfusepath{stroke}
\end{pgfscope}
\end{pgfscope}
\end{pgfscope}
\end{pgfscope}
\end{tikzpicture}}} \right) \right] - \tau^{- \frac{1 + \nu}{2}} \left(
     \left[ 3 \tau^{\frac{1 + \nu}{2}} (- X^{\!\resizebox{0.6em}{!}{
\begin{tikzpicture}
\pgfpathmoveto{\pgfqpoint{0cm}{-0.035cm}}
\pgfpathlineto{\pgfqpoint{1.376cm}{-0.035cm}}
\pgfpathlineto{\pgfqpoint{1.376cm}{1.552cm}}
\pgfpathlineto{\pgfqpoint{0cm}{1.552cm}}
\pgfpathclose
\pgfusepath{clip}
\begin{pgfscope}
\begin{pgfscope}
\pgfpathmoveto{\pgfqpoint{0cm}{-0.035cm}}
\pgfpathlineto{\pgfqpoint{1.376cm}{-0.035cm}}
\pgfpathlineto{\pgfqpoint{1.376cm}{1.552cm}}
\pgfpathlineto{\pgfqpoint{0cm}{1.552cm}}
\pgfpathclose
\pgfusepath{clip}
\begin{pgfscope}
\begin{pgfscope}
\pgfsetdash{}{0cm}
\pgfsetlinewidth{0.818mm}
\pgfsetroundcap
\pgfsetroundjoin
\pgfsetmiterlimit{7.0}
\definecolor{eps2pgf_color}{gray}{0}\pgfsetstrokecolor{eps2pgf_color}\pgfsetfillcolor{eps2pgf_color}
\pgfpathmoveto{\pgfqpoint{0.117cm}{1.421cm}}
\pgfpathlineto{\pgfqpoint{0.682cm}{0.671cm}}
\pgfpathlineto{\pgfqpoint{1.246cm}{1.421cm}}
\pgfusepath{stroke}
\end{pgfscope}
\definecolor{eps2pgf_color}{gray}{0}\pgfsetstrokecolor{eps2pgf_color}\pgfsetfillcolor{eps2pgf_color}
\pgfpathmoveto{\pgfqpoint{0.273cm}{1.395cm}}
\pgfpathcurveto{\pgfqpoint{0.273cm}{1.432cm}}{\pgfqpoint{0.259cm}{1.467cm}}{\pgfqpoint{0.233cm}{1.492cm}}
\pgfpathcurveto{\pgfqpoint{0.207cm}{1.518cm}}{\pgfqpoint{0.173cm}{1.532cm}}{\pgfqpoint{0.137cm}{1.532cm}}
\pgfpathcurveto{\pgfqpoint{0.1cm}{1.532cm}}{\pgfqpoint{0.066cm}{1.518cm}}{\pgfqpoint{0.04cm}{1.492cm}}
\pgfpathcurveto{\pgfqpoint{0.014cm}{1.467cm}}{\pgfqpoint{0cm}{1.432cm}}{\pgfqpoint{0cm}{1.395cm}}
\pgfpathcurveto{\pgfqpoint{0cm}{1.359cm}}{\pgfqpoint{0.014cm}{1.324cm}}{\pgfqpoint{0.04cm}{1.299cm}}
\pgfpathcurveto{\pgfqpoint{0.066cm}{1.273cm}}{\pgfqpoint{0.1cm}{1.258cm}}{\pgfqpoint{0.137cm}{1.258cm}}
\pgfpathcurveto{\pgfqpoint{0.173cm}{1.258cm}}{\pgfqpoint{0.207cm}{1.273cm}}{\pgfqpoint{0.233cm}{1.299cm}}
\pgfpathcurveto{\pgfqpoint{0.259cm}{1.324cm}}{\pgfqpoint{0.273cm}{1.359cm}}{\pgfqpoint{0.273cm}{1.395cm}}
\pgfusepath{fill}
\begin{pgfscope}
\pgfsetdash{}{0cm}
\pgfsetlinewidth{0.818mm}
\pgfsetmiterlimit{7.0}
\pgfpathmoveto{\pgfqpoint{0.682cm}{0.671cm}}
\pgfpathlineto{\pgfqpoint{0.679cm}{1.418cm}}
\pgfusepath{stroke}
\end{pgfscope}
\pgfpathmoveto{\pgfqpoint{0.815cm}{1.399cm}}
\pgfpathcurveto{\pgfqpoint{0.815cm}{1.435cm}}{\pgfqpoint{0.801cm}{1.47cm}}{\pgfqpoint{0.775cm}{1.496cm}}
\pgfpathcurveto{\pgfqpoint{0.75cm}{1.521cm}}{\pgfqpoint{0.715cm}{1.536cm}}{\pgfqpoint{0.679cm}{1.536cm}}
\pgfpathcurveto{\pgfqpoint{0.643cm}{1.536cm}}{\pgfqpoint{0.608cm}{1.521cm}}{\pgfqpoint{0.582cm}{1.496cm}}
\pgfpathcurveto{\pgfqpoint{0.557cm}{1.47cm}}{\pgfqpoint{0.542cm}{1.435cm}}{\pgfqpoint{0.542cm}{1.399cm}}
\pgfpathcurveto{\pgfqpoint{0.542cm}{1.363cm}}{\pgfqpoint{0.557cm}{1.328cm}}{\pgfqpoint{0.582cm}{1.302cm}}
\pgfpathcurveto{\pgfqpoint{0.608cm}{1.276cm}}{\pgfqpoint{0.643cm}{1.262cm}}{\pgfqpoint{0.679cm}{1.262cm}}
\pgfpathcurveto{\pgfqpoint{0.715cm}{1.262cm}}{\pgfqpoint{0.75cm}{1.276cm}}{\pgfqpoint{0.775cm}{1.302cm}}
\pgfpathcurveto{\pgfqpoint{0.801cm}{1.328cm}}{\pgfqpoint{0.815cm}{1.363cm}}{\pgfqpoint{0.815cm}{1.399cm}}
\pgfusepath{fill}
\pgfpathmoveto{\pgfqpoint{1.345cm}{1.371cm}}
\pgfpathcurveto{\pgfqpoint{1.345cm}{1.408cm}}{\pgfqpoint{1.331cm}{1.442cm}}{\pgfqpoint{1.305cm}{1.468cm}}
\pgfpathcurveto{\pgfqpoint{1.28cm}{1.494cm}}{\pgfqpoint{1.245cm}{1.508cm}}{\pgfqpoint{1.209cm}{1.508cm}}
\pgfpathcurveto{\pgfqpoint{1.172cm}{1.508cm}}{\pgfqpoint{1.138cm}{1.494cm}}{\pgfqpoint{1.112cm}{1.468cm}}
\pgfpathcurveto{\pgfqpoint{1.087cm}{1.442cm}}{\pgfqpoint{1.072cm}{1.408cm}}{\pgfqpoint{1.072cm}{1.371cm}}
\pgfpathcurveto{\pgfqpoint{1.072cm}{1.335cm}}{\pgfqpoint{1.087cm}{1.3cm}}{\pgfqpoint{1.112cm}{1.274cm}}
\pgfpathcurveto{\pgfqpoint{1.138cm}{1.249cm}}{\pgfqpoint{1.172cm}{1.234cm}}{\pgfqpoint{1.209cm}{1.234cm}}
\pgfpathcurveto{\pgfqpoint{1.245cm}{1.234cm}}{\pgfqpoint{1.28cm}{1.249cm}}{\pgfqpoint{1.305cm}{1.274cm}}
\pgfpathcurveto{\pgfqpoint{1.331cm}{1.3cm}}{\pgfqpoint{1.345cm}{1.335cm}}{\pgfqpoint{1.345cm}{1.371cm}}
\pgfusepath{fill}
\begin{pgfscope}
\pgfsetdash{}{0cm}
\pgfsetlinewidth{0.818mm}
\pgfsetroundcap
\pgfsetmiterlimit{4.0}
\pgfpathmoveto{\pgfqpoint{0.682cm}{0.671cm}}
\pgfpathlineto{\pgfqpoint{0.682cm}{0.042cm}}
\pgfusepath{stroke}
\end{pgfscope}
\end{pgfscope}
\end{pgfscope}
\end{pgfscope}
\end{tikzpicture}}} + \phi + \psi) \right]
     \precprec\llbracket X^2\rrbracket \right) \right) \]
  \[ + 3 X (- X^{\!\resizebox{0.6em}{!}{
\begin{tikzpicture}
\pgfpathmoveto{\pgfqpoint{0cm}{-0.035cm}}
\pgfpathlineto{\pgfqpoint{1.376cm}{-0.035cm}}
\pgfpathlineto{\pgfqpoint{1.376cm}{1.552cm}}
\pgfpathlineto{\pgfqpoint{0cm}{1.552cm}}
\pgfpathclose
\pgfusepath{clip}
\begin{pgfscope}
\begin{pgfscope}
\pgfpathmoveto{\pgfqpoint{0cm}{-0.035cm}}
\pgfpathlineto{\pgfqpoint{1.376cm}{-0.035cm}}
\pgfpathlineto{\pgfqpoint{1.376cm}{1.552cm}}
\pgfpathlineto{\pgfqpoint{0cm}{1.552cm}}
\pgfpathclose
\pgfusepath{clip}
\begin{pgfscope}
\begin{pgfscope}
\pgfsetdash{}{0cm}
\pgfsetlinewidth{0.818mm}
\pgfsetroundcap
\pgfsetroundjoin
\pgfsetmiterlimit{7.0}
\definecolor{eps2pgf_color}{gray}{0}\pgfsetstrokecolor{eps2pgf_color}\pgfsetfillcolor{eps2pgf_color}
\pgfpathmoveto{\pgfqpoint{0.117cm}{1.421cm}}
\pgfpathlineto{\pgfqpoint{0.682cm}{0.671cm}}
\pgfpathlineto{\pgfqpoint{1.246cm}{1.421cm}}
\pgfusepath{stroke}
\end{pgfscope}
\definecolor{eps2pgf_color}{gray}{0}\pgfsetstrokecolor{eps2pgf_color}\pgfsetfillcolor{eps2pgf_color}
\pgfpathmoveto{\pgfqpoint{0.273cm}{1.395cm}}
\pgfpathcurveto{\pgfqpoint{0.273cm}{1.432cm}}{\pgfqpoint{0.259cm}{1.467cm}}{\pgfqpoint{0.233cm}{1.492cm}}
\pgfpathcurveto{\pgfqpoint{0.207cm}{1.518cm}}{\pgfqpoint{0.173cm}{1.532cm}}{\pgfqpoint{0.137cm}{1.532cm}}
\pgfpathcurveto{\pgfqpoint{0.1cm}{1.532cm}}{\pgfqpoint{0.066cm}{1.518cm}}{\pgfqpoint{0.04cm}{1.492cm}}
\pgfpathcurveto{\pgfqpoint{0.014cm}{1.467cm}}{\pgfqpoint{0cm}{1.432cm}}{\pgfqpoint{0cm}{1.395cm}}
\pgfpathcurveto{\pgfqpoint{0cm}{1.359cm}}{\pgfqpoint{0.014cm}{1.324cm}}{\pgfqpoint{0.04cm}{1.299cm}}
\pgfpathcurveto{\pgfqpoint{0.066cm}{1.273cm}}{\pgfqpoint{0.1cm}{1.258cm}}{\pgfqpoint{0.137cm}{1.258cm}}
\pgfpathcurveto{\pgfqpoint{0.173cm}{1.258cm}}{\pgfqpoint{0.207cm}{1.273cm}}{\pgfqpoint{0.233cm}{1.299cm}}
\pgfpathcurveto{\pgfqpoint{0.259cm}{1.324cm}}{\pgfqpoint{0.273cm}{1.359cm}}{\pgfqpoint{0.273cm}{1.395cm}}
\pgfusepath{fill}
\begin{pgfscope}
\pgfsetdash{}{0cm}
\pgfsetlinewidth{0.818mm}
\pgfsetmiterlimit{7.0}
\pgfpathmoveto{\pgfqpoint{0.682cm}{0.671cm}}
\pgfpathlineto{\pgfqpoint{0.679cm}{1.418cm}}
\pgfusepath{stroke}
\end{pgfscope}
\pgfpathmoveto{\pgfqpoint{0.815cm}{1.399cm}}
\pgfpathcurveto{\pgfqpoint{0.815cm}{1.435cm}}{\pgfqpoint{0.801cm}{1.47cm}}{\pgfqpoint{0.775cm}{1.496cm}}
\pgfpathcurveto{\pgfqpoint{0.75cm}{1.521cm}}{\pgfqpoint{0.715cm}{1.536cm}}{\pgfqpoint{0.679cm}{1.536cm}}
\pgfpathcurveto{\pgfqpoint{0.643cm}{1.536cm}}{\pgfqpoint{0.608cm}{1.521cm}}{\pgfqpoint{0.582cm}{1.496cm}}
\pgfpathcurveto{\pgfqpoint{0.557cm}{1.47cm}}{\pgfqpoint{0.542cm}{1.435cm}}{\pgfqpoint{0.542cm}{1.399cm}}
\pgfpathcurveto{\pgfqpoint{0.542cm}{1.363cm}}{\pgfqpoint{0.557cm}{1.328cm}}{\pgfqpoint{0.582cm}{1.302cm}}
\pgfpathcurveto{\pgfqpoint{0.608cm}{1.276cm}}{\pgfqpoint{0.643cm}{1.262cm}}{\pgfqpoint{0.679cm}{1.262cm}}
\pgfpathcurveto{\pgfqpoint{0.715cm}{1.262cm}}{\pgfqpoint{0.75cm}{1.276cm}}{\pgfqpoint{0.775cm}{1.302cm}}
\pgfpathcurveto{\pgfqpoint{0.801cm}{1.328cm}}{\pgfqpoint{0.815cm}{1.363cm}}{\pgfqpoint{0.815cm}{1.399cm}}
\pgfusepath{fill}
\pgfpathmoveto{\pgfqpoint{1.345cm}{1.371cm}}
\pgfpathcurveto{\pgfqpoint{1.345cm}{1.408cm}}{\pgfqpoint{1.331cm}{1.442cm}}{\pgfqpoint{1.305cm}{1.468cm}}
\pgfpathcurveto{\pgfqpoint{1.28cm}{1.494cm}}{\pgfqpoint{1.245cm}{1.508cm}}{\pgfqpoint{1.209cm}{1.508cm}}
\pgfpathcurveto{\pgfqpoint{1.172cm}{1.508cm}}{\pgfqpoint{1.138cm}{1.494cm}}{\pgfqpoint{1.112cm}{1.468cm}}
\pgfpathcurveto{\pgfqpoint{1.087cm}{1.442cm}}{\pgfqpoint{1.072cm}{1.408cm}}{\pgfqpoint{1.072cm}{1.371cm}}
\pgfpathcurveto{\pgfqpoint{1.072cm}{1.335cm}}{\pgfqpoint{1.087cm}{1.3cm}}{\pgfqpoint{1.112cm}{1.274cm}}
\pgfpathcurveto{\pgfqpoint{1.138cm}{1.249cm}}{\pgfqpoint{1.172cm}{1.234cm}}{\pgfqpoint{1.209cm}{1.234cm}}
\pgfpathcurveto{\pgfqpoint{1.245cm}{1.234cm}}{\pgfqpoint{1.28cm}{1.249cm}}{\pgfqpoint{1.305cm}{1.274cm}}
\pgfpathcurveto{\pgfqpoint{1.331cm}{1.3cm}}{\pgfqpoint{1.345cm}{1.335cm}}{\pgfqpoint{1.345cm}{1.371cm}}
\pgfusepath{fill}
\begin{pgfscope}
\pgfsetdash{}{0cm}
\pgfsetlinewidth{0.818mm}
\pgfsetroundcap
\pgfsetmiterlimit{4.0}
\pgfpathmoveto{\pgfqpoint{0.682cm}{0.671cm}}
\pgfpathlineto{\pgfqpoint{0.682cm}{0.042cm}}
\pgfusepath{stroke}
\end{pgfscope}
\end{pgfscope}
\end{pgfscope}
\end{pgfscope}
\end{tikzpicture}}} + \phi + \psi)^2 + (- X^{\!\resizebox{0.6em}{!}{
\begin{tikzpicture}
\pgfpathmoveto{\pgfqpoint{0cm}{-0.035cm}}
\pgfpathlineto{\pgfqpoint{1.376cm}{-0.035cm}}
\pgfpathlineto{\pgfqpoint{1.376cm}{1.552cm}}
\pgfpathlineto{\pgfqpoint{0cm}{1.552cm}}
\pgfpathclose
\pgfusepath{clip}
\begin{pgfscope}
\begin{pgfscope}
\pgfpathmoveto{\pgfqpoint{0cm}{-0.035cm}}
\pgfpathlineto{\pgfqpoint{1.376cm}{-0.035cm}}
\pgfpathlineto{\pgfqpoint{1.376cm}{1.552cm}}
\pgfpathlineto{\pgfqpoint{0cm}{1.552cm}}
\pgfpathclose
\pgfusepath{clip}
\begin{pgfscope}
\begin{pgfscope}
\pgfsetdash{}{0cm}
\pgfsetlinewidth{0.818mm}
\pgfsetroundcap
\pgfsetroundjoin
\pgfsetmiterlimit{7.0}
\definecolor{eps2pgf_color}{gray}{0}\pgfsetstrokecolor{eps2pgf_color}\pgfsetfillcolor{eps2pgf_color}
\pgfpathmoveto{\pgfqpoint{0.117cm}{1.421cm}}
\pgfpathlineto{\pgfqpoint{0.682cm}{0.671cm}}
\pgfpathlineto{\pgfqpoint{1.246cm}{1.421cm}}
\pgfusepath{stroke}
\end{pgfscope}
\definecolor{eps2pgf_color}{gray}{0}\pgfsetstrokecolor{eps2pgf_color}\pgfsetfillcolor{eps2pgf_color}
\pgfpathmoveto{\pgfqpoint{0.273cm}{1.395cm}}
\pgfpathcurveto{\pgfqpoint{0.273cm}{1.432cm}}{\pgfqpoint{0.259cm}{1.467cm}}{\pgfqpoint{0.233cm}{1.492cm}}
\pgfpathcurveto{\pgfqpoint{0.207cm}{1.518cm}}{\pgfqpoint{0.173cm}{1.532cm}}{\pgfqpoint{0.137cm}{1.532cm}}
\pgfpathcurveto{\pgfqpoint{0.1cm}{1.532cm}}{\pgfqpoint{0.066cm}{1.518cm}}{\pgfqpoint{0.04cm}{1.492cm}}
\pgfpathcurveto{\pgfqpoint{0.014cm}{1.467cm}}{\pgfqpoint{0cm}{1.432cm}}{\pgfqpoint{0cm}{1.395cm}}
\pgfpathcurveto{\pgfqpoint{0cm}{1.359cm}}{\pgfqpoint{0.014cm}{1.324cm}}{\pgfqpoint{0.04cm}{1.299cm}}
\pgfpathcurveto{\pgfqpoint{0.066cm}{1.273cm}}{\pgfqpoint{0.1cm}{1.258cm}}{\pgfqpoint{0.137cm}{1.258cm}}
\pgfpathcurveto{\pgfqpoint{0.173cm}{1.258cm}}{\pgfqpoint{0.207cm}{1.273cm}}{\pgfqpoint{0.233cm}{1.299cm}}
\pgfpathcurveto{\pgfqpoint{0.259cm}{1.324cm}}{\pgfqpoint{0.273cm}{1.359cm}}{\pgfqpoint{0.273cm}{1.395cm}}
\pgfusepath{fill}
\begin{pgfscope}
\pgfsetdash{}{0cm}
\pgfsetlinewidth{0.818mm}
\pgfsetmiterlimit{7.0}
\pgfpathmoveto{\pgfqpoint{0.682cm}{0.671cm}}
\pgfpathlineto{\pgfqpoint{0.679cm}{1.418cm}}
\pgfusepath{stroke}
\end{pgfscope}
\pgfpathmoveto{\pgfqpoint{0.815cm}{1.399cm}}
\pgfpathcurveto{\pgfqpoint{0.815cm}{1.435cm}}{\pgfqpoint{0.801cm}{1.47cm}}{\pgfqpoint{0.775cm}{1.496cm}}
\pgfpathcurveto{\pgfqpoint{0.75cm}{1.521cm}}{\pgfqpoint{0.715cm}{1.536cm}}{\pgfqpoint{0.679cm}{1.536cm}}
\pgfpathcurveto{\pgfqpoint{0.643cm}{1.536cm}}{\pgfqpoint{0.608cm}{1.521cm}}{\pgfqpoint{0.582cm}{1.496cm}}
\pgfpathcurveto{\pgfqpoint{0.557cm}{1.47cm}}{\pgfqpoint{0.542cm}{1.435cm}}{\pgfqpoint{0.542cm}{1.399cm}}
\pgfpathcurveto{\pgfqpoint{0.542cm}{1.363cm}}{\pgfqpoint{0.557cm}{1.328cm}}{\pgfqpoint{0.582cm}{1.302cm}}
\pgfpathcurveto{\pgfqpoint{0.608cm}{1.276cm}}{\pgfqpoint{0.643cm}{1.262cm}}{\pgfqpoint{0.679cm}{1.262cm}}
\pgfpathcurveto{\pgfqpoint{0.715cm}{1.262cm}}{\pgfqpoint{0.75cm}{1.276cm}}{\pgfqpoint{0.775cm}{1.302cm}}
\pgfpathcurveto{\pgfqpoint{0.801cm}{1.328cm}}{\pgfqpoint{0.815cm}{1.363cm}}{\pgfqpoint{0.815cm}{1.399cm}}
\pgfusepath{fill}
\pgfpathmoveto{\pgfqpoint{1.345cm}{1.371cm}}
\pgfpathcurveto{\pgfqpoint{1.345cm}{1.408cm}}{\pgfqpoint{1.331cm}{1.442cm}}{\pgfqpoint{1.305cm}{1.468cm}}
\pgfpathcurveto{\pgfqpoint{1.28cm}{1.494cm}}{\pgfqpoint{1.245cm}{1.508cm}}{\pgfqpoint{1.209cm}{1.508cm}}
\pgfpathcurveto{\pgfqpoint{1.172cm}{1.508cm}}{\pgfqpoint{1.138cm}{1.494cm}}{\pgfqpoint{1.112cm}{1.468cm}}
\pgfpathcurveto{\pgfqpoint{1.087cm}{1.442cm}}{\pgfqpoint{1.072cm}{1.408cm}}{\pgfqpoint{1.072cm}{1.371cm}}
\pgfpathcurveto{\pgfqpoint{1.072cm}{1.335cm}}{\pgfqpoint{1.087cm}{1.3cm}}{\pgfqpoint{1.112cm}{1.274cm}}
\pgfpathcurveto{\pgfqpoint{1.138cm}{1.249cm}}{\pgfqpoint{1.172cm}{1.234cm}}{\pgfqpoint{1.209cm}{1.234cm}}
\pgfpathcurveto{\pgfqpoint{1.245cm}{1.234cm}}{\pgfqpoint{1.28cm}{1.249cm}}{\pgfqpoint{1.305cm}{1.274cm}}
\pgfpathcurveto{\pgfqpoint{1.331cm}{1.3cm}}{\pgfqpoint{1.345cm}{1.335cm}}{\pgfqpoint{1.345cm}{1.371cm}}
\pgfusepath{fill}
\begin{pgfscope}
\pgfsetdash{}{0cm}
\pgfsetlinewidth{0.818mm}
\pgfsetroundcap
\pgfsetmiterlimit{4.0}
\pgfpathmoveto{\pgfqpoint{0.682cm}{0.671cm}}
\pgfpathlineto{\pgfqpoint{0.682cm}{0.042cm}}
\pgfusepath{stroke}
\end{pgfscope}
\end{pgfscope}
\end{pgfscope}
\end{pgfscope}
\end{tikzpicture}}} + \phi +
     \psi)^3 + 3 b \varphi \]
  as a replacement of \eqref{eq:rhs43a}. Similarly
  to Section \ref{sec:43}, we obtain $\LL \breve{\vartheta} +
  \breve{\Theta} = 0$ with $\breve{\vartheta}(0)=\phi(0)$ and where $\breve{\Theta}$ differs from $\Theta$ only in the
  two commutators above, which we denote by $\mathrm{com_{1}},\,\mathrm{com_{2}}$ according to their order of appearance.
  
  Recall that in Section \ref{sec:45} and Section
  \ref{sec:43}, a suitable bound for $\vartheta$ was only needed in order to
  control the term $\llbracket X^2\rrbracket  \circ \vartheta$ in the equation for $\psi$. To
  this end, the choice of the weight $\rho^{2 + \alpha}$ in Section
  \ref{ssec:theta} was rather arbitrary, which has already been observed in Section \ref{sec:43}. Indeed, in order to control $\llbracket X^2\rrbracket 
  \circ \vartheta$ in $\CC^{\gamma} (\rho^{3 + \gamma})$ (cf. Section
  \ref{ssec:psi1}) it is sufficient to bound $\vartheta$ in
  $\CC^{1 + \alpha} (\rho^{3 + \gamma'})$ for some $0 < \gamma' < \gamma$.
 Similarly to Section \ref{sec:43}, we make use of this flexibility here: our goal is to apply Lemma
  \ref{lemma:tau-weighted-schauder} in order to estimate $\breve{\vartheta}$
  in $C \CC^{1 + \alpha}_{} (\tau^{(2 + \alpha) / 2} \rho_{}^{3 + \gamma'})$. This higher weight will allow us to control the term coming from $\Psi$ in the bound for time regularity below, see \eqref{eq:tregpsi}.

  Let us first estimate the new commutators appearing in $\breve{\Theta} .$ We
  rewrite the first commutator as
  \[ \tmop{com}_1= \tau^{- \frac{1 + \nu}{2}} \left( 3 \left[
     \tau^{\frac{1 + \nu}{2}} (- X^{\!\resizebox{0.6em}{!}{
\begin{tikzpicture}
\pgfpathmoveto{\pgfqpoint{0cm}{-0.035cm}}
\pgfpathlineto{\pgfqpoint{1.376cm}{-0.035cm}}
\pgfpathlineto{\pgfqpoint{1.376cm}{1.552cm}}
\pgfpathlineto{\pgfqpoint{0cm}{1.552cm}}
\pgfpathclose
\pgfusepath{clip}
\begin{pgfscope}
\begin{pgfscope}
\pgfpathmoveto{\pgfqpoint{0cm}{-0.035cm}}
\pgfpathlineto{\pgfqpoint{1.376cm}{-0.035cm}}
\pgfpathlineto{\pgfqpoint{1.376cm}{1.552cm}}
\pgfpathlineto{\pgfqpoint{0cm}{1.552cm}}
\pgfpathclose
\pgfusepath{clip}
\begin{pgfscope}
\begin{pgfscope}
\pgfsetdash{}{0cm}
\pgfsetlinewidth{0.818mm}
\pgfsetroundcap
\pgfsetroundjoin
\pgfsetmiterlimit{7.0}
\definecolor{eps2pgf_color}{gray}{0}\pgfsetstrokecolor{eps2pgf_color}\pgfsetfillcolor{eps2pgf_color}
\pgfpathmoveto{\pgfqpoint{0.117cm}{1.421cm}}
\pgfpathlineto{\pgfqpoint{0.682cm}{0.671cm}}
\pgfpathlineto{\pgfqpoint{1.246cm}{1.421cm}}
\pgfusepath{stroke}
\end{pgfscope}
\definecolor{eps2pgf_color}{gray}{0}\pgfsetstrokecolor{eps2pgf_color}\pgfsetfillcolor{eps2pgf_color}
\pgfpathmoveto{\pgfqpoint{0.273cm}{1.395cm}}
\pgfpathcurveto{\pgfqpoint{0.273cm}{1.432cm}}{\pgfqpoint{0.259cm}{1.467cm}}{\pgfqpoint{0.233cm}{1.492cm}}
\pgfpathcurveto{\pgfqpoint{0.207cm}{1.518cm}}{\pgfqpoint{0.173cm}{1.532cm}}{\pgfqpoint{0.137cm}{1.532cm}}
\pgfpathcurveto{\pgfqpoint{0.1cm}{1.532cm}}{\pgfqpoint{0.066cm}{1.518cm}}{\pgfqpoint{0.04cm}{1.492cm}}
\pgfpathcurveto{\pgfqpoint{0.014cm}{1.467cm}}{\pgfqpoint{0cm}{1.432cm}}{\pgfqpoint{0cm}{1.395cm}}
\pgfpathcurveto{\pgfqpoint{0cm}{1.359cm}}{\pgfqpoint{0.014cm}{1.324cm}}{\pgfqpoint{0.04cm}{1.299cm}}
\pgfpathcurveto{\pgfqpoint{0.066cm}{1.273cm}}{\pgfqpoint{0.1cm}{1.258cm}}{\pgfqpoint{0.137cm}{1.258cm}}
\pgfpathcurveto{\pgfqpoint{0.173cm}{1.258cm}}{\pgfqpoint{0.207cm}{1.273cm}}{\pgfqpoint{0.233cm}{1.299cm}}
\pgfpathcurveto{\pgfqpoint{0.259cm}{1.324cm}}{\pgfqpoint{0.273cm}{1.359cm}}{\pgfqpoint{0.273cm}{1.395cm}}
\pgfusepath{fill}
\begin{pgfscope}
\pgfsetdash{}{0cm}
\pgfsetlinewidth{0.818mm}
\pgfsetmiterlimit{7.0}
\pgfpathmoveto{\pgfqpoint{0.682cm}{0.671cm}}
\pgfpathlineto{\pgfqpoint{0.679cm}{1.418cm}}
\pgfusepath{stroke}
\end{pgfscope}
\pgfpathmoveto{\pgfqpoint{0.815cm}{1.399cm}}
\pgfpathcurveto{\pgfqpoint{0.815cm}{1.435cm}}{\pgfqpoint{0.801cm}{1.47cm}}{\pgfqpoint{0.775cm}{1.496cm}}
\pgfpathcurveto{\pgfqpoint{0.75cm}{1.521cm}}{\pgfqpoint{0.715cm}{1.536cm}}{\pgfqpoint{0.679cm}{1.536cm}}
\pgfpathcurveto{\pgfqpoint{0.643cm}{1.536cm}}{\pgfqpoint{0.608cm}{1.521cm}}{\pgfqpoint{0.582cm}{1.496cm}}
\pgfpathcurveto{\pgfqpoint{0.557cm}{1.47cm}}{\pgfqpoint{0.542cm}{1.435cm}}{\pgfqpoint{0.542cm}{1.399cm}}
\pgfpathcurveto{\pgfqpoint{0.542cm}{1.363cm}}{\pgfqpoint{0.557cm}{1.328cm}}{\pgfqpoint{0.582cm}{1.302cm}}
\pgfpathcurveto{\pgfqpoint{0.608cm}{1.276cm}}{\pgfqpoint{0.643cm}{1.262cm}}{\pgfqpoint{0.679cm}{1.262cm}}
\pgfpathcurveto{\pgfqpoint{0.715cm}{1.262cm}}{\pgfqpoint{0.75cm}{1.276cm}}{\pgfqpoint{0.775cm}{1.302cm}}
\pgfpathcurveto{\pgfqpoint{0.801cm}{1.328cm}}{\pgfqpoint{0.815cm}{1.363cm}}{\pgfqpoint{0.815cm}{1.399cm}}
\pgfusepath{fill}
\pgfpathmoveto{\pgfqpoint{1.345cm}{1.371cm}}
\pgfpathcurveto{\pgfqpoint{1.345cm}{1.408cm}}{\pgfqpoint{1.331cm}{1.442cm}}{\pgfqpoint{1.305cm}{1.468cm}}
\pgfpathcurveto{\pgfqpoint{1.28cm}{1.494cm}}{\pgfqpoint{1.245cm}{1.508cm}}{\pgfqpoint{1.209cm}{1.508cm}}
\pgfpathcurveto{\pgfqpoint{1.172cm}{1.508cm}}{\pgfqpoint{1.138cm}{1.494cm}}{\pgfqpoint{1.112cm}{1.468cm}}
\pgfpathcurveto{\pgfqpoint{1.087cm}{1.442cm}}{\pgfqpoint{1.072cm}{1.408cm}}{\pgfqpoint{1.072cm}{1.371cm}}
\pgfpathcurveto{\pgfqpoint{1.072cm}{1.335cm}}{\pgfqpoint{1.087cm}{1.3cm}}{\pgfqpoint{1.112cm}{1.274cm}}
\pgfpathcurveto{\pgfqpoint{1.138cm}{1.249cm}}{\pgfqpoint{1.172cm}{1.234cm}}{\pgfqpoint{1.209cm}{1.234cm}}
\pgfpathcurveto{\pgfqpoint{1.245cm}{1.234cm}}{\pgfqpoint{1.28cm}{1.249cm}}{\pgfqpoint{1.305cm}{1.274cm}}
\pgfpathcurveto{\pgfqpoint{1.331cm}{1.3cm}}{\pgfqpoint{1.345cm}{1.335cm}}{\pgfqpoint{1.345cm}{1.371cm}}
\pgfusepath{fill}
\begin{pgfscope}
\pgfsetdash{}{0cm}
\pgfsetlinewidth{0.818mm}
\pgfsetroundcap
\pgfsetmiterlimit{4.0}
\pgfpathmoveto{\pgfqpoint{0.682cm}{0.671cm}}
\pgfpathlineto{\pgfqpoint{0.682cm}{0.042cm}}
\pgfusepath{stroke}
\end{pgfscope}
\end{pgfscope}
\end{pgfscope}
\end{pgfscope}
\end{tikzpicture}}} + \phi + \psi) \right]
     \precprec \llbracket X^2\rrbracket - 3 \left[ \tau^{\frac{1 + \nu}{2}} (- X^{\!\resizebox{0.6em}{!}{
\begin{tikzpicture}
\pgfpathmoveto{\pgfqpoint{0cm}{-0.035cm}}
\pgfpathlineto{\pgfqpoint{1.376cm}{-0.035cm}}
\pgfpathlineto{\pgfqpoint{1.376cm}{1.552cm}}
\pgfpathlineto{\pgfqpoint{0cm}{1.552cm}}
\pgfpathclose
\pgfusepath{clip}
\begin{pgfscope}
\begin{pgfscope}
\pgfpathmoveto{\pgfqpoint{0cm}{-0.035cm}}
\pgfpathlineto{\pgfqpoint{1.376cm}{-0.035cm}}
\pgfpathlineto{\pgfqpoint{1.376cm}{1.552cm}}
\pgfpathlineto{\pgfqpoint{0cm}{1.552cm}}
\pgfpathclose
\pgfusepath{clip}
\begin{pgfscope}
\begin{pgfscope}
\pgfsetdash{}{0cm}
\pgfsetlinewidth{0.818mm}
\pgfsetroundcap
\pgfsetroundjoin
\pgfsetmiterlimit{7.0}
\definecolor{eps2pgf_color}{gray}{0}\pgfsetstrokecolor{eps2pgf_color}\pgfsetfillcolor{eps2pgf_color}
\pgfpathmoveto{\pgfqpoint{0.117cm}{1.421cm}}
\pgfpathlineto{\pgfqpoint{0.682cm}{0.671cm}}
\pgfpathlineto{\pgfqpoint{1.246cm}{1.421cm}}
\pgfusepath{stroke}
\end{pgfscope}
\definecolor{eps2pgf_color}{gray}{0}\pgfsetstrokecolor{eps2pgf_color}\pgfsetfillcolor{eps2pgf_color}
\pgfpathmoveto{\pgfqpoint{0.273cm}{1.395cm}}
\pgfpathcurveto{\pgfqpoint{0.273cm}{1.432cm}}{\pgfqpoint{0.259cm}{1.467cm}}{\pgfqpoint{0.233cm}{1.492cm}}
\pgfpathcurveto{\pgfqpoint{0.207cm}{1.518cm}}{\pgfqpoint{0.173cm}{1.532cm}}{\pgfqpoint{0.137cm}{1.532cm}}
\pgfpathcurveto{\pgfqpoint{0.1cm}{1.532cm}}{\pgfqpoint{0.066cm}{1.518cm}}{\pgfqpoint{0.04cm}{1.492cm}}
\pgfpathcurveto{\pgfqpoint{0.014cm}{1.467cm}}{\pgfqpoint{0cm}{1.432cm}}{\pgfqpoint{0cm}{1.395cm}}
\pgfpathcurveto{\pgfqpoint{0cm}{1.359cm}}{\pgfqpoint{0.014cm}{1.324cm}}{\pgfqpoint{0.04cm}{1.299cm}}
\pgfpathcurveto{\pgfqpoint{0.066cm}{1.273cm}}{\pgfqpoint{0.1cm}{1.258cm}}{\pgfqpoint{0.137cm}{1.258cm}}
\pgfpathcurveto{\pgfqpoint{0.173cm}{1.258cm}}{\pgfqpoint{0.207cm}{1.273cm}}{\pgfqpoint{0.233cm}{1.299cm}}
\pgfpathcurveto{\pgfqpoint{0.259cm}{1.324cm}}{\pgfqpoint{0.273cm}{1.359cm}}{\pgfqpoint{0.273cm}{1.395cm}}
\pgfusepath{fill}
\begin{pgfscope}
\pgfsetdash{}{0cm}
\pgfsetlinewidth{0.818mm}
\pgfsetmiterlimit{7.0}
\pgfpathmoveto{\pgfqpoint{0.682cm}{0.671cm}}
\pgfpathlineto{\pgfqpoint{0.679cm}{1.418cm}}
\pgfusepath{stroke}
\end{pgfscope}
\pgfpathmoveto{\pgfqpoint{0.815cm}{1.399cm}}
\pgfpathcurveto{\pgfqpoint{0.815cm}{1.435cm}}{\pgfqpoint{0.801cm}{1.47cm}}{\pgfqpoint{0.775cm}{1.496cm}}
\pgfpathcurveto{\pgfqpoint{0.75cm}{1.521cm}}{\pgfqpoint{0.715cm}{1.536cm}}{\pgfqpoint{0.679cm}{1.536cm}}
\pgfpathcurveto{\pgfqpoint{0.643cm}{1.536cm}}{\pgfqpoint{0.608cm}{1.521cm}}{\pgfqpoint{0.582cm}{1.496cm}}
\pgfpathcurveto{\pgfqpoint{0.557cm}{1.47cm}}{\pgfqpoint{0.542cm}{1.435cm}}{\pgfqpoint{0.542cm}{1.399cm}}
\pgfpathcurveto{\pgfqpoint{0.542cm}{1.363cm}}{\pgfqpoint{0.557cm}{1.328cm}}{\pgfqpoint{0.582cm}{1.302cm}}
\pgfpathcurveto{\pgfqpoint{0.608cm}{1.276cm}}{\pgfqpoint{0.643cm}{1.262cm}}{\pgfqpoint{0.679cm}{1.262cm}}
\pgfpathcurveto{\pgfqpoint{0.715cm}{1.262cm}}{\pgfqpoint{0.75cm}{1.276cm}}{\pgfqpoint{0.775cm}{1.302cm}}
\pgfpathcurveto{\pgfqpoint{0.801cm}{1.328cm}}{\pgfqpoint{0.815cm}{1.363cm}}{\pgfqpoint{0.815cm}{1.399cm}}
\pgfusepath{fill}
\pgfpathmoveto{\pgfqpoint{1.345cm}{1.371cm}}
\pgfpathcurveto{\pgfqpoint{1.345cm}{1.408cm}}{\pgfqpoint{1.331cm}{1.442cm}}{\pgfqpoint{1.305cm}{1.468cm}}
\pgfpathcurveto{\pgfqpoint{1.28cm}{1.494cm}}{\pgfqpoint{1.245cm}{1.508cm}}{\pgfqpoint{1.209cm}{1.508cm}}
\pgfpathcurveto{\pgfqpoint{1.172cm}{1.508cm}}{\pgfqpoint{1.138cm}{1.494cm}}{\pgfqpoint{1.112cm}{1.468cm}}
\pgfpathcurveto{\pgfqpoint{1.087cm}{1.442cm}}{\pgfqpoint{1.072cm}{1.408cm}}{\pgfqpoint{1.072cm}{1.371cm}}
\pgfpathcurveto{\pgfqpoint{1.072cm}{1.335cm}}{\pgfqpoint{1.087cm}{1.3cm}}{\pgfqpoint{1.112cm}{1.274cm}}
\pgfpathcurveto{\pgfqpoint{1.138cm}{1.249cm}}{\pgfqpoint{1.172cm}{1.234cm}}{\pgfqpoint{1.209cm}{1.234cm}}
\pgfpathcurveto{\pgfqpoint{1.245cm}{1.234cm}}{\pgfqpoint{1.28cm}{1.249cm}}{\pgfqpoint{1.305cm}{1.274cm}}
\pgfpathcurveto{\pgfqpoint{1.331cm}{1.3cm}}{\pgfqpoint{1.345cm}{1.335cm}}{\pgfqpoint{1.345cm}{1.371cm}}
\pgfusepath{fill}
\begin{pgfscope}
\pgfsetdash{}{0cm}
\pgfsetlinewidth{0.818mm}
\pgfsetroundcap
\pgfsetmiterlimit{4.0}
\pgfpathmoveto{\pgfqpoint{0.682cm}{0.671cm}}
\pgfpathlineto{\pgfqpoint{0.682cm}{0.042cm}}
\pgfusepath{stroke}
\end{pgfscope}
\end{pgfscope}
\end{pgfscope}
\end{pgfscope}
\end{tikzpicture}}} +
     \phi + \psi) \right] \prec \llbracket X^2\rrbracket \right), \]
  and observe that by Lemma \ref{lem:5.1} it can be estimated in $\CC^{- 1 +
  \alpha} (\rho_{}^{3 + \gamma'})$, using the time regularity of
  $\tau^{\frac{1 + \nu}{2}} (- X^{\!\resizebox{0.6em}{!}{
\begin{tikzpicture}
\pgfpathmoveto{\pgfqpoint{0cm}{-0.035cm}}
\pgfpathlineto{\pgfqpoint{1.376cm}{-0.035cm}}
\pgfpathlineto{\pgfqpoint{1.376cm}{1.552cm}}
\pgfpathlineto{\pgfqpoint{0cm}{1.552cm}}
\pgfpathclose
\pgfusepath{clip}
\begin{pgfscope}
\begin{pgfscope}
\pgfpathmoveto{\pgfqpoint{0cm}{-0.035cm}}
\pgfpathlineto{\pgfqpoint{1.376cm}{-0.035cm}}
\pgfpathlineto{\pgfqpoint{1.376cm}{1.552cm}}
\pgfpathlineto{\pgfqpoint{0cm}{1.552cm}}
\pgfpathclose
\pgfusepath{clip}
\begin{pgfscope}
\begin{pgfscope}
\pgfsetdash{}{0cm}
\pgfsetlinewidth{0.818mm}
\pgfsetroundcap
\pgfsetroundjoin
\pgfsetmiterlimit{7.0}
\definecolor{eps2pgf_color}{gray}{0}\pgfsetstrokecolor{eps2pgf_color}\pgfsetfillcolor{eps2pgf_color}
\pgfpathmoveto{\pgfqpoint{0.117cm}{1.421cm}}
\pgfpathlineto{\pgfqpoint{0.682cm}{0.671cm}}
\pgfpathlineto{\pgfqpoint{1.246cm}{1.421cm}}
\pgfusepath{stroke}
\end{pgfscope}
\definecolor{eps2pgf_color}{gray}{0}\pgfsetstrokecolor{eps2pgf_color}\pgfsetfillcolor{eps2pgf_color}
\pgfpathmoveto{\pgfqpoint{0.273cm}{1.395cm}}
\pgfpathcurveto{\pgfqpoint{0.273cm}{1.432cm}}{\pgfqpoint{0.259cm}{1.467cm}}{\pgfqpoint{0.233cm}{1.492cm}}
\pgfpathcurveto{\pgfqpoint{0.207cm}{1.518cm}}{\pgfqpoint{0.173cm}{1.532cm}}{\pgfqpoint{0.137cm}{1.532cm}}
\pgfpathcurveto{\pgfqpoint{0.1cm}{1.532cm}}{\pgfqpoint{0.066cm}{1.518cm}}{\pgfqpoint{0.04cm}{1.492cm}}
\pgfpathcurveto{\pgfqpoint{0.014cm}{1.467cm}}{\pgfqpoint{0cm}{1.432cm}}{\pgfqpoint{0cm}{1.395cm}}
\pgfpathcurveto{\pgfqpoint{0cm}{1.359cm}}{\pgfqpoint{0.014cm}{1.324cm}}{\pgfqpoint{0.04cm}{1.299cm}}
\pgfpathcurveto{\pgfqpoint{0.066cm}{1.273cm}}{\pgfqpoint{0.1cm}{1.258cm}}{\pgfqpoint{0.137cm}{1.258cm}}
\pgfpathcurveto{\pgfqpoint{0.173cm}{1.258cm}}{\pgfqpoint{0.207cm}{1.273cm}}{\pgfqpoint{0.233cm}{1.299cm}}
\pgfpathcurveto{\pgfqpoint{0.259cm}{1.324cm}}{\pgfqpoint{0.273cm}{1.359cm}}{\pgfqpoint{0.273cm}{1.395cm}}
\pgfusepath{fill}
\begin{pgfscope}
\pgfsetdash{}{0cm}
\pgfsetlinewidth{0.818mm}
\pgfsetmiterlimit{7.0}
\pgfpathmoveto{\pgfqpoint{0.682cm}{0.671cm}}
\pgfpathlineto{\pgfqpoint{0.679cm}{1.418cm}}
\pgfusepath{stroke}
\end{pgfscope}
\pgfpathmoveto{\pgfqpoint{0.815cm}{1.399cm}}
\pgfpathcurveto{\pgfqpoint{0.815cm}{1.435cm}}{\pgfqpoint{0.801cm}{1.47cm}}{\pgfqpoint{0.775cm}{1.496cm}}
\pgfpathcurveto{\pgfqpoint{0.75cm}{1.521cm}}{\pgfqpoint{0.715cm}{1.536cm}}{\pgfqpoint{0.679cm}{1.536cm}}
\pgfpathcurveto{\pgfqpoint{0.643cm}{1.536cm}}{\pgfqpoint{0.608cm}{1.521cm}}{\pgfqpoint{0.582cm}{1.496cm}}
\pgfpathcurveto{\pgfqpoint{0.557cm}{1.47cm}}{\pgfqpoint{0.542cm}{1.435cm}}{\pgfqpoint{0.542cm}{1.399cm}}
\pgfpathcurveto{\pgfqpoint{0.542cm}{1.363cm}}{\pgfqpoint{0.557cm}{1.328cm}}{\pgfqpoint{0.582cm}{1.302cm}}
\pgfpathcurveto{\pgfqpoint{0.608cm}{1.276cm}}{\pgfqpoint{0.643cm}{1.262cm}}{\pgfqpoint{0.679cm}{1.262cm}}
\pgfpathcurveto{\pgfqpoint{0.715cm}{1.262cm}}{\pgfqpoint{0.75cm}{1.276cm}}{\pgfqpoint{0.775cm}{1.302cm}}
\pgfpathcurveto{\pgfqpoint{0.801cm}{1.328cm}}{\pgfqpoint{0.815cm}{1.363cm}}{\pgfqpoint{0.815cm}{1.399cm}}
\pgfusepath{fill}
\pgfpathmoveto{\pgfqpoint{1.345cm}{1.371cm}}
\pgfpathcurveto{\pgfqpoint{1.345cm}{1.408cm}}{\pgfqpoint{1.331cm}{1.442cm}}{\pgfqpoint{1.305cm}{1.468cm}}
\pgfpathcurveto{\pgfqpoint{1.28cm}{1.494cm}}{\pgfqpoint{1.245cm}{1.508cm}}{\pgfqpoint{1.209cm}{1.508cm}}
\pgfpathcurveto{\pgfqpoint{1.172cm}{1.508cm}}{\pgfqpoint{1.138cm}{1.494cm}}{\pgfqpoint{1.112cm}{1.468cm}}
\pgfpathcurveto{\pgfqpoint{1.087cm}{1.442cm}}{\pgfqpoint{1.072cm}{1.408cm}}{\pgfqpoint{1.072cm}{1.371cm}}
\pgfpathcurveto{\pgfqpoint{1.072cm}{1.335cm}}{\pgfqpoint{1.087cm}{1.3cm}}{\pgfqpoint{1.112cm}{1.274cm}}
\pgfpathcurveto{\pgfqpoint{1.138cm}{1.249cm}}{\pgfqpoint{1.172cm}{1.234cm}}{\pgfqpoint{1.209cm}{1.234cm}}
\pgfpathcurveto{\pgfqpoint{1.245cm}{1.234cm}}{\pgfqpoint{1.28cm}{1.249cm}}{\pgfqpoint{1.305cm}{1.274cm}}
\pgfpathcurveto{\pgfqpoint{1.331cm}{1.3cm}}{\pgfqpoint{1.345cm}{1.335cm}}{\pgfqpoint{1.345cm}{1.371cm}}
\pgfusepath{fill}
\begin{pgfscope}
\pgfsetdash{}{0cm}
\pgfsetlinewidth{0.818mm}
\pgfsetroundcap
\pgfsetmiterlimit{4.0}
\pgfpathmoveto{\pgfqpoint{0.682cm}{0.671cm}}
\pgfpathlineto{\pgfqpoint{0.682cm}{0.042cm}}
\pgfusepath{stroke}
\end{pgfscope}
\end{pgfscope}
\end{pgfscope}
\end{pgfscope}
\end{tikzpicture}}} + \phi + \psi)$, provided we can control the blow-up in time by a suitable power of the $\tau$ weight. Hence, in view of Lemma
  \ref{lemma:tau-weighted-schauder} (with $\beta_i = 0$), which we aim to apply in order to gain the required regularity of $\breve{\vartheta}$, we  estimate
  (provided $2 + \alpha \geq 1 + \nu$)
  \[ \| \tmop{com}_1 \|_{C \CC^{- 1 + \alpha} (\tau^{(2 + \alpha) / 2}
     \rho_{}^{3 + \gamma'})^{}} \]
  \[ \lesssim \left\| \left[ \tau^{\frac{1 + \nu}{2}} (- X^{\!\resizebox{0.6em}{!}{
\begin{tikzpicture}
\pgfpathmoveto{\pgfqpoint{0cm}{-0.035cm}}
\pgfpathlineto{\pgfqpoint{1.376cm}{-0.035cm}}
\pgfpathlineto{\pgfqpoint{1.376cm}{1.552cm}}
\pgfpathlineto{\pgfqpoint{0cm}{1.552cm}}
\pgfpathclose
\pgfusepath{clip}
\begin{pgfscope}
\begin{pgfscope}
\pgfpathmoveto{\pgfqpoint{0cm}{-0.035cm}}
\pgfpathlineto{\pgfqpoint{1.376cm}{-0.035cm}}
\pgfpathlineto{\pgfqpoint{1.376cm}{1.552cm}}
\pgfpathlineto{\pgfqpoint{0cm}{1.552cm}}
\pgfpathclose
\pgfusepath{clip}
\begin{pgfscope}
\begin{pgfscope}
\pgfsetdash{}{0cm}
\pgfsetlinewidth{0.818mm}
\pgfsetroundcap
\pgfsetroundjoin
\pgfsetmiterlimit{7.0}
\definecolor{eps2pgf_color}{gray}{0}\pgfsetstrokecolor{eps2pgf_color}\pgfsetfillcolor{eps2pgf_color}
\pgfpathmoveto{\pgfqpoint{0.117cm}{1.421cm}}
\pgfpathlineto{\pgfqpoint{0.682cm}{0.671cm}}
\pgfpathlineto{\pgfqpoint{1.246cm}{1.421cm}}
\pgfusepath{stroke}
\end{pgfscope}
\definecolor{eps2pgf_color}{gray}{0}\pgfsetstrokecolor{eps2pgf_color}\pgfsetfillcolor{eps2pgf_color}
\pgfpathmoveto{\pgfqpoint{0.273cm}{1.395cm}}
\pgfpathcurveto{\pgfqpoint{0.273cm}{1.432cm}}{\pgfqpoint{0.259cm}{1.467cm}}{\pgfqpoint{0.233cm}{1.492cm}}
\pgfpathcurveto{\pgfqpoint{0.207cm}{1.518cm}}{\pgfqpoint{0.173cm}{1.532cm}}{\pgfqpoint{0.137cm}{1.532cm}}
\pgfpathcurveto{\pgfqpoint{0.1cm}{1.532cm}}{\pgfqpoint{0.066cm}{1.518cm}}{\pgfqpoint{0.04cm}{1.492cm}}
\pgfpathcurveto{\pgfqpoint{0.014cm}{1.467cm}}{\pgfqpoint{0cm}{1.432cm}}{\pgfqpoint{0cm}{1.395cm}}
\pgfpathcurveto{\pgfqpoint{0cm}{1.359cm}}{\pgfqpoint{0.014cm}{1.324cm}}{\pgfqpoint{0.04cm}{1.299cm}}
\pgfpathcurveto{\pgfqpoint{0.066cm}{1.273cm}}{\pgfqpoint{0.1cm}{1.258cm}}{\pgfqpoint{0.137cm}{1.258cm}}
\pgfpathcurveto{\pgfqpoint{0.173cm}{1.258cm}}{\pgfqpoint{0.207cm}{1.273cm}}{\pgfqpoint{0.233cm}{1.299cm}}
\pgfpathcurveto{\pgfqpoint{0.259cm}{1.324cm}}{\pgfqpoint{0.273cm}{1.359cm}}{\pgfqpoint{0.273cm}{1.395cm}}
\pgfusepath{fill}
\begin{pgfscope}
\pgfsetdash{}{0cm}
\pgfsetlinewidth{0.818mm}
\pgfsetmiterlimit{7.0}
\pgfpathmoveto{\pgfqpoint{0.682cm}{0.671cm}}
\pgfpathlineto{\pgfqpoint{0.679cm}{1.418cm}}
\pgfusepath{stroke}
\end{pgfscope}
\pgfpathmoveto{\pgfqpoint{0.815cm}{1.399cm}}
\pgfpathcurveto{\pgfqpoint{0.815cm}{1.435cm}}{\pgfqpoint{0.801cm}{1.47cm}}{\pgfqpoint{0.775cm}{1.496cm}}
\pgfpathcurveto{\pgfqpoint{0.75cm}{1.521cm}}{\pgfqpoint{0.715cm}{1.536cm}}{\pgfqpoint{0.679cm}{1.536cm}}
\pgfpathcurveto{\pgfqpoint{0.643cm}{1.536cm}}{\pgfqpoint{0.608cm}{1.521cm}}{\pgfqpoint{0.582cm}{1.496cm}}
\pgfpathcurveto{\pgfqpoint{0.557cm}{1.47cm}}{\pgfqpoint{0.542cm}{1.435cm}}{\pgfqpoint{0.542cm}{1.399cm}}
\pgfpathcurveto{\pgfqpoint{0.542cm}{1.363cm}}{\pgfqpoint{0.557cm}{1.328cm}}{\pgfqpoint{0.582cm}{1.302cm}}
\pgfpathcurveto{\pgfqpoint{0.608cm}{1.276cm}}{\pgfqpoint{0.643cm}{1.262cm}}{\pgfqpoint{0.679cm}{1.262cm}}
\pgfpathcurveto{\pgfqpoint{0.715cm}{1.262cm}}{\pgfqpoint{0.75cm}{1.276cm}}{\pgfqpoint{0.775cm}{1.302cm}}
\pgfpathcurveto{\pgfqpoint{0.801cm}{1.328cm}}{\pgfqpoint{0.815cm}{1.363cm}}{\pgfqpoint{0.815cm}{1.399cm}}
\pgfusepath{fill}
\pgfpathmoveto{\pgfqpoint{1.345cm}{1.371cm}}
\pgfpathcurveto{\pgfqpoint{1.345cm}{1.408cm}}{\pgfqpoint{1.331cm}{1.442cm}}{\pgfqpoint{1.305cm}{1.468cm}}
\pgfpathcurveto{\pgfqpoint{1.28cm}{1.494cm}}{\pgfqpoint{1.245cm}{1.508cm}}{\pgfqpoint{1.209cm}{1.508cm}}
\pgfpathcurveto{\pgfqpoint{1.172cm}{1.508cm}}{\pgfqpoint{1.138cm}{1.494cm}}{\pgfqpoint{1.112cm}{1.468cm}}
\pgfpathcurveto{\pgfqpoint{1.087cm}{1.442cm}}{\pgfqpoint{1.072cm}{1.408cm}}{\pgfqpoint{1.072cm}{1.371cm}}
\pgfpathcurveto{\pgfqpoint{1.072cm}{1.335cm}}{\pgfqpoint{1.087cm}{1.3cm}}{\pgfqpoint{1.112cm}{1.274cm}}
\pgfpathcurveto{\pgfqpoint{1.138cm}{1.249cm}}{\pgfqpoint{1.172cm}{1.234cm}}{\pgfqpoint{1.209cm}{1.234cm}}
\pgfpathcurveto{\pgfqpoint{1.245cm}{1.234cm}}{\pgfqpoint{1.28cm}{1.249cm}}{\pgfqpoint{1.305cm}{1.274cm}}
\pgfpathcurveto{\pgfqpoint{1.331cm}{1.3cm}}{\pgfqpoint{1.345cm}{1.335cm}}{\pgfqpoint{1.345cm}{1.371cm}}
\pgfusepath{fill}
\begin{pgfscope}
\pgfsetdash{}{0cm}
\pgfsetlinewidth{0.818mm}
\pgfsetroundcap
\pgfsetmiterlimit{4.0}
\pgfpathmoveto{\pgfqpoint{0.682cm}{0.671cm}}
\pgfpathlineto{\pgfqpoint{0.682cm}{0.042cm}}
\pgfusepath{stroke}
\end{pgfscope}
\end{pgfscope}
\end{pgfscope}
\end{pgfscope}
\end{tikzpicture}}} + \phi
     + \psi) \right] \precprec \llbracket X^2\rrbracket- 3 \left[ \tau^{\frac{1 + \nu}{2}} (-
     X^{\!\resizebox{0.6em}{!}{
\begin{tikzpicture}
\pgfpathmoveto{\pgfqpoint{0cm}{-0.035cm}}
\pgfpathlineto{\pgfqpoint{1.376cm}{-0.035cm}}
\pgfpathlineto{\pgfqpoint{1.376cm}{1.552cm}}
\pgfpathlineto{\pgfqpoint{0cm}{1.552cm}}
\pgfpathclose
\pgfusepath{clip}
\begin{pgfscope}
\begin{pgfscope}
\pgfpathmoveto{\pgfqpoint{0cm}{-0.035cm}}
\pgfpathlineto{\pgfqpoint{1.376cm}{-0.035cm}}
\pgfpathlineto{\pgfqpoint{1.376cm}{1.552cm}}
\pgfpathlineto{\pgfqpoint{0cm}{1.552cm}}
\pgfpathclose
\pgfusepath{clip}
\begin{pgfscope}
\begin{pgfscope}
\pgfsetdash{}{0cm}
\pgfsetlinewidth{0.818mm}
\pgfsetroundcap
\pgfsetroundjoin
\pgfsetmiterlimit{7.0}
\definecolor{eps2pgf_color}{gray}{0}\pgfsetstrokecolor{eps2pgf_color}\pgfsetfillcolor{eps2pgf_color}
\pgfpathmoveto{\pgfqpoint{0.117cm}{1.421cm}}
\pgfpathlineto{\pgfqpoint{0.682cm}{0.671cm}}
\pgfpathlineto{\pgfqpoint{1.246cm}{1.421cm}}
\pgfusepath{stroke}
\end{pgfscope}
\definecolor{eps2pgf_color}{gray}{0}\pgfsetstrokecolor{eps2pgf_color}\pgfsetfillcolor{eps2pgf_color}
\pgfpathmoveto{\pgfqpoint{0.273cm}{1.395cm}}
\pgfpathcurveto{\pgfqpoint{0.273cm}{1.432cm}}{\pgfqpoint{0.259cm}{1.467cm}}{\pgfqpoint{0.233cm}{1.492cm}}
\pgfpathcurveto{\pgfqpoint{0.207cm}{1.518cm}}{\pgfqpoint{0.173cm}{1.532cm}}{\pgfqpoint{0.137cm}{1.532cm}}
\pgfpathcurveto{\pgfqpoint{0.1cm}{1.532cm}}{\pgfqpoint{0.066cm}{1.518cm}}{\pgfqpoint{0.04cm}{1.492cm}}
\pgfpathcurveto{\pgfqpoint{0.014cm}{1.467cm}}{\pgfqpoint{0cm}{1.432cm}}{\pgfqpoint{0cm}{1.395cm}}
\pgfpathcurveto{\pgfqpoint{0cm}{1.359cm}}{\pgfqpoint{0.014cm}{1.324cm}}{\pgfqpoint{0.04cm}{1.299cm}}
\pgfpathcurveto{\pgfqpoint{0.066cm}{1.273cm}}{\pgfqpoint{0.1cm}{1.258cm}}{\pgfqpoint{0.137cm}{1.258cm}}
\pgfpathcurveto{\pgfqpoint{0.173cm}{1.258cm}}{\pgfqpoint{0.207cm}{1.273cm}}{\pgfqpoint{0.233cm}{1.299cm}}
\pgfpathcurveto{\pgfqpoint{0.259cm}{1.324cm}}{\pgfqpoint{0.273cm}{1.359cm}}{\pgfqpoint{0.273cm}{1.395cm}}
\pgfusepath{fill}
\begin{pgfscope}
\pgfsetdash{}{0cm}
\pgfsetlinewidth{0.818mm}
\pgfsetmiterlimit{7.0}
\pgfpathmoveto{\pgfqpoint{0.682cm}{0.671cm}}
\pgfpathlineto{\pgfqpoint{0.679cm}{1.418cm}}
\pgfusepath{stroke}
\end{pgfscope}
\pgfpathmoveto{\pgfqpoint{0.815cm}{1.399cm}}
\pgfpathcurveto{\pgfqpoint{0.815cm}{1.435cm}}{\pgfqpoint{0.801cm}{1.47cm}}{\pgfqpoint{0.775cm}{1.496cm}}
\pgfpathcurveto{\pgfqpoint{0.75cm}{1.521cm}}{\pgfqpoint{0.715cm}{1.536cm}}{\pgfqpoint{0.679cm}{1.536cm}}
\pgfpathcurveto{\pgfqpoint{0.643cm}{1.536cm}}{\pgfqpoint{0.608cm}{1.521cm}}{\pgfqpoint{0.582cm}{1.496cm}}
\pgfpathcurveto{\pgfqpoint{0.557cm}{1.47cm}}{\pgfqpoint{0.542cm}{1.435cm}}{\pgfqpoint{0.542cm}{1.399cm}}
\pgfpathcurveto{\pgfqpoint{0.542cm}{1.363cm}}{\pgfqpoint{0.557cm}{1.328cm}}{\pgfqpoint{0.582cm}{1.302cm}}
\pgfpathcurveto{\pgfqpoint{0.608cm}{1.276cm}}{\pgfqpoint{0.643cm}{1.262cm}}{\pgfqpoint{0.679cm}{1.262cm}}
\pgfpathcurveto{\pgfqpoint{0.715cm}{1.262cm}}{\pgfqpoint{0.75cm}{1.276cm}}{\pgfqpoint{0.775cm}{1.302cm}}
\pgfpathcurveto{\pgfqpoint{0.801cm}{1.328cm}}{\pgfqpoint{0.815cm}{1.363cm}}{\pgfqpoint{0.815cm}{1.399cm}}
\pgfusepath{fill}
\pgfpathmoveto{\pgfqpoint{1.345cm}{1.371cm}}
\pgfpathcurveto{\pgfqpoint{1.345cm}{1.408cm}}{\pgfqpoint{1.331cm}{1.442cm}}{\pgfqpoint{1.305cm}{1.468cm}}
\pgfpathcurveto{\pgfqpoint{1.28cm}{1.494cm}}{\pgfqpoint{1.245cm}{1.508cm}}{\pgfqpoint{1.209cm}{1.508cm}}
\pgfpathcurveto{\pgfqpoint{1.172cm}{1.508cm}}{\pgfqpoint{1.138cm}{1.494cm}}{\pgfqpoint{1.112cm}{1.468cm}}
\pgfpathcurveto{\pgfqpoint{1.087cm}{1.442cm}}{\pgfqpoint{1.072cm}{1.408cm}}{\pgfqpoint{1.072cm}{1.371cm}}
\pgfpathcurveto{\pgfqpoint{1.072cm}{1.335cm}}{\pgfqpoint{1.087cm}{1.3cm}}{\pgfqpoint{1.112cm}{1.274cm}}
\pgfpathcurveto{\pgfqpoint{1.138cm}{1.249cm}}{\pgfqpoint{1.172cm}{1.234cm}}{\pgfqpoint{1.209cm}{1.234cm}}
\pgfpathcurveto{\pgfqpoint{1.245cm}{1.234cm}}{\pgfqpoint{1.28cm}{1.249cm}}{\pgfqpoint{1.305cm}{1.274cm}}
\pgfpathcurveto{\pgfqpoint{1.331cm}{1.3cm}}{\pgfqpoint{1.345cm}{1.335cm}}{\pgfqpoint{1.345cm}{1.371cm}}
\pgfusepath{fill}
\begin{pgfscope}
\pgfsetdash{}{0cm}
\pgfsetlinewidth{0.818mm}
\pgfsetroundcap
\pgfsetmiterlimit{4.0}
\pgfpathmoveto{\pgfqpoint{0.682cm}{0.671cm}}
\pgfpathlineto{\pgfqpoint{0.682cm}{0.042cm}}
\pgfusepath{stroke}
\end{pgfscope}
\end{pgfscope}
\end{pgfscope}
\end{pgfscope}
\end{tikzpicture}}} + \phi + \psi) \right] \prec\llbracket X^2\rrbracket \right\|_{C \CC^{- 1 +
     \alpha} (\rho_{}^{3 + \gamma'})^{}} \]
  \[ \lesssim 1 + \left\| \tau^{\frac{1 + \nu}{2}} (\phi + \psi)
     \right\|_{C^{(\alpha + \kappa) / 2} L^{\infty} (\rho^{3 + \gamma''})} \]
  for some $0 < \gamma'' < \gamma'$. For the second commutator, it holds
  \[ \tmop{com}_2= - \frac{1 + \nu}{2} \tau^{- \frac{1 + \nu}{2} - 1}
     (1 - \tau) \left[ \left( 3 \tau^{\frac{1 + \nu}{2}} (- X^{\!\resizebox{0.6em}{!}{
\begin{tikzpicture}
\pgfpathmoveto{\pgfqpoint{0cm}{-0.035cm}}
\pgfpathlineto{\pgfqpoint{1.376cm}{-0.035cm}}
\pgfpathlineto{\pgfqpoint{1.376cm}{1.552cm}}
\pgfpathlineto{\pgfqpoint{0cm}{1.552cm}}
\pgfpathclose
\pgfusepath{clip}
\begin{pgfscope}
\begin{pgfscope}
\pgfpathmoveto{\pgfqpoint{0cm}{-0.035cm}}
\pgfpathlineto{\pgfqpoint{1.376cm}{-0.035cm}}
\pgfpathlineto{\pgfqpoint{1.376cm}{1.552cm}}
\pgfpathlineto{\pgfqpoint{0cm}{1.552cm}}
\pgfpathclose
\pgfusepath{clip}
\begin{pgfscope}
\begin{pgfscope}
\pgfsetdash{}{0cm}
\pgfsetlinewidth{0.818mm}
\pgfsetroundcap
\pgfsetroundjoin
\pgfsetmiterlimit{7.0}
\definecolor{eps2pgf_color}{gray}{0}\pgfsetstrokecolor{eps2pgf_color}\pgfsetfillcolor{eps2pgf_color}
\pgfpathmoveto{\pgfqpoint{0.117cm}{1.421cm}}
\pgfpathlineto{\pgfqpoint{0.682cm}{0.671cm}}
\pgfpathlineto{\pgfqpoint{1.246cm}{1.421cm}}
\pgfusepath{stroke}
\end{pgfscope}
\definecolor{eps2pgf_color}{gray}{0}\pgfsetstrokecolor{eps2pgf_color}\pgfsetfillcolor{eps2pgf_color}
\pgfpathmoveto{\pgfqpoint{0.273cm}{1.395cm}}
\pgfpathcurveto{\pgfqpoint{0.273cm}{1.432cm}}{\pgfqpoint{0.259cm}{1.467cm}}{\pgfqpoint{0.233cm}{1.492cm}}
\pgfpathcurveto{\pgfqpoint{0.207cm}{1.518cm}}{\pgfqpoint{0.173cm}{1.532cm}}{\pgfqpoint{0.137cm}{1.532cm}}
\pgfpathcurveto{\pgfqpoint{0.1cm}{1.532cm}}{\pgfqpoint{0.066cm}{1.518cm}}{\pgfqpoint{0.04cm}{1.492cm}}
\pgfpathcurveto{\pgfqpoint{0.014cm}{1.467cm}}{\pgfqpoint{0cm}{1.432cm}}{\pgfqpoint{0cm}{1.395cm}}
\pgfpathcurveto{\pgfqpoint{0cm}{1.359cm}}{\pgfqpoint{0.014cm}{1.324cm}}{\pgfqpoint{0.04cm}{1.299cm}}
\pgfpathcurveto{\pgfqpoint{0.066cm}{1.273cm}}{\pgfqpoint{0.1cm}{1.258cm}}{\pgfqpoint{0.137cm}{1.258cm}}
\pgfpathcurveto{\pgfqpoint{0.173cm}{1.258cm}}{\pgfqpoint{0.207cm}{1.273cm}}{\pgfqpoint{0.233cm}{1.299cm}}
\pgfpathcurveto{\pgfqpoint{0.259cm}{1.324cm}}{\pgfqpoint{0.273cm}{1.359cm}}{\pgfqpoint{0.273cm}{1.395cm}}
\pgfusepath{fill}
\begin{pgfscope}
\pgfsetdash{}{0cm}
\pgfsetlinewidth{0.818mm}
\pgfsetmiterlimit{7.0}
\pgfpathmoveto{\pgfqpoint{0.682cm}{0.671cm}}
\pgfpathlineto{\pgfqpoint{0.679cm}{1.418cm}}
\pgfusepath{stroke}
\end{pgfscope}
\pgfpathmoveto{\pgfqpoint{0.815cm}{1.399cm}}
\pgfpathcurveto{\pgfqpoint{0.815cm}{1.435cm}}{\pgfqpoint{0.801cm}{1.47cm}}{\pgfqpoint{0.775cm}{1.496cm}}
\pgfpathcurveto{\pgfqpoint{0.75cm}{1.521cm}}{\pgfqpoint{0.715cm}{1.536cm}}{\pgfqpoint{0.679cm}{1.536cm}}
\pgfpathcurveto{\pgfqpoint{0.643cm}{1.536cm}}{\pgfqpoint{0.608cm}{1.521cm}}{\pgfqpoint{0.582cm}{1.496cm}}
\pgfpathcurveto{\pgfqpoint{0.557cm}{1.47cm}}{\pgfqpoint{0.542cm}{1.435cm}}{\pgfqpoint{0.542cm}{1.399cm}}
\pgfpathcurveto{\pgfqpoint{0.542cm}{1.363cm}}{\pgfqpoint{0.557cm}{1.328cm}}{\pgfqpoint{0.582cm}{1.302cm}}
\pgfpathcurveto{\pgfqpoint{0.608cm}{1.276cm}}{\pgfqpoint{0.643cm}{1.262cm}}{\pgfqpoint{0.679cm}{1.262cm}}
\pgfpathcurveto{\pgfqpoint{0.715cm}{1.262cm}}{\pgfqpoint{0.75cm}{1.276cm}}{\pgfqpoint{0.775cm}{1.302cm}}
\pgfpathcurveto{\pgfqpoint{0.801cm}{1.328cm}}{\pgfqpoint{0.815cm}{1.363cm}}{\pgfqpoint{0.815cm}{1.399cm}}
\pgfusepath{fill}
\pgfpathmoveto{\pgfqpoint{1.345cm}{1.371cm}}
\pgfpathcurveto{\pgfqpoint{1.345cm}{1.408cm}}{\pgfqpoint{1.331cm}{1.442cm}}{\pgfqpoint{1.305cm}{1.468cm}}
\pgfpathcurveto{\pgfqpoint{1.28cm}{1.494cm}}{\pgfqpoint{1.245cm}{1.508cm}}{\pgfqpoint{1.209cm}{1.508cm}}
\pgfpathcurveto{\pgfqpoint{1.172cm}{1.508cm}}{\pgfqpoint{1.138cm}{1.494cm}}{\pgfqpoint{1.112cm}{1.468cm}}
\pgfpathcurveto{\pgfqpoint{1.087cm}{1.442cm}}{\pgfqpoint{1.072cm}{1.408cm}}{\pgfqpoint{1.072cm}{1.371cm}}
\pgfpathcurveto{\pgfqpoint{1.072cm}{1.335cm}}{\pgfqpoint{1.087cm}{1.3cm}}{\pgfqpoint{1.112cm}{1.274cm}}
\pgfpathcurveto{\pgfqpoint{1.138cm}{1.249cm}}{\pgfqpoint{1.172cm}{1.234cm}}{\pgfqpoint{1.209cm}{1.234cm}}
\pgfpathcurveto{\pgfqpoint{1.245cm}{1.234cm}}{\pgfqpoint{1.28cm}{1.249cm}}{\pgfqpoint{1.305cm}{1.274cm}}
\pgfpathcurveto{\pgfqpoint{1.331cm}{1.3cm}}{\pgfqpoint{1.345cm}{1.335cm}}{\pgfqpoint{1.345cm}{1.371cm}}
\pgfusepath{fill}
\begin{pgfscope}
\pgfsetdash{}{0cm}
\pgfsetlinewidth{0.818mm}
\pgfsetroundcap
\pgfsetmiterlimit{4.0}
\pgfpathmoveto{\pgfqpoint{0.682cm}{0.671cm}}
\pgfpathlineto{\pgfqpoint{0.682cm}{0.042cm}}
\pgfusepath{stroke}
\end{pgfscope}
\end{pgfscope}
\end{pgfscope}
\end{pgfscope}
\end{tikzpicture}}} +
     \phi + \psi) \right) \precprec X^{\!\resizebox{0.6em}{!}{
\begin{tikzpicture}
\pgfpathmoveto{\pgfqpoint{0cm}{0cm}}
\pgfpathlineto{\pgfqpoint{1.376cm}{0cm}}
\pgfpathlineto{\pgfqpoint{1.376cm}{1.588cm}}
\pgfpathlineto{\pgfqpoint{0cm}{1.588cm}}
\pgfpathclose
\pgfusepath{clip}
\begin{pgfscope}
\begin{pgfscope}
\pgfpathmoveto{\pgfqpoint{0cm}{0cm}}
\pgfpathlineto{\pgfqpoint{1.376cm}{0cm}}
\pgfpathlineto{\pgfqpoint{1.376cm}{1.588cm}}
\pgfpathlineto{\pgfqpoint{0cm}{1.588cm}}
\pgfpathclose
\pgfusepath{clip}
\begin{pgfscope}
\begin{pgfscope}
\definecolor{eps2pgf_color}{gray}{0.976471}\pgfsetstrokecolor{eps2pgf_color}\pgfsetfillcolor{eps2pgf_color}
\pgfpathmoveto{\pgfqpoint{0cm}{0cm}}
\pgfpathlineto{\pgfqpoint{1.376cm}{0cm}}
\pgfpathlineto{\pgfqpoint{1.376cm}{1.588cm}}
\pgfpathlineto{\pgfqpoint{0cm}{1.588cm}}
\pgfpathclose
\pgfusepath{fill}
\end{pgfscope}
\begin{pgfscope}
\pgfsetdash{}{0cm}
\pgfsetlinewidth{0.818mm}
\pgfsetroundcap
\pgfsetroundjoin
\pgfsetmiterlimit{7.0}
\definecolor{eps2pgf_color}{gray}{0}\pgfsetstrokecolor{eps2pgf_color}\pgfsetfillcolor{eps2pgf_color}
\pgfpathmoveto{\pgfqpoint{0.117cm}{1.476cm}}
\pgfpathlineto{\pgfqpoint{0.682cm}{0.726cm}}
\pgfpathlineto{\pgfqpoint{1.246cm}{1.476cm}}
\pgfusepath{stroke}
\end{pgfscope}
\definecolor{eps2pgf_color}{gray}{0}\pgfsetstrokecolor{eps2pgf_color}\pgfsetfillcolor{eps2pgf_color}
\pgfpathmoveto{\pgfqpoint{0.273cm}{1.451cm}}
\pgfpathcurveto{\pgfqpoint{0.273cm}{1.487cm}}{\pgfqpoint{0.259cm}{1.522cm}}{\pgfqpoint{0.233cm}{1.547cm}}
\pgfpathcurveto{\pgfqpoint{0.207cm}{1.573cm}}{\pgfqpoint{0.173cm}{1.588cm}}{\pgfqpoint{0.137cm}{1.588cm}}
\pgfpathcurveto{\pgfqpoint{0.1cm}{1.588cm}}{\pgfqpoint{0.066cm}{1.573cm}}{\pgfqpoint{0.04cm}{1.547cm}}
\pgfpathcurveto{\pgfqpoint{0.014cm}{1.522cm}}{\pgfqpoint{0cm}{1.487cm}}{\pgfqpoint{0cm}{1.451cm}}
\pgfpathcurveto{\pgfqpoint{0cm}{1.414cm}}{\pgfqpoint{0.014cm}{1.379cm}}{\pgfqpoint{0.04cm}{1.354cm}}
\pgfpathcurveto{\pgfqpoint{0.066cm}{1.328cm}}{\pgfqpoint{0.1cm}{1.314cm}}{\pgfqpoint{0.137cm}{1.314cm}}
\pgfpathcurveto{\pgfqpoint{0.173cm}{1.314cm}}{\pgfqpoint{0.207cm}{1.328cm}}{\pgfqpoint{0.233cm}{1.354cm}}
\pgfpathcurveto{\pgfqpoint{0.259cm}{1.379cm}}{\pgfqpoint{0.273cm}{1.414cm}}{\pgfqpoint{0.273cm}{1.451cm}}
\pgfusepath{fill}
\pgfpathmoveto{\pgfqpoint{1.345cm}{1.426cm}}
\pgfpathcurveto{\pgfqpoint{1.345cm}{1.463cm}}{\pgfqpoint{1.331cm}{1.497cm}}{\pgfqpoint{1.305cm}{1.523cm}}
\pgfpathcurveto{\pgfqpoint{1.28cm}{1.549cm}}{\pgfqpoint{1.245cm}{1.563cm}}{\pgfqpoint{1.209cm}{1.563cm}}
\pgfpathcurveto{\pgfqpoint{1.172cm}{1.563cm}}{\pgfqpoint{1.138cm}{1.549cm}}{\pgfqpoint{1.112cm}{1.523cm}}
\pgfpathcurveto{\pgfqpoint{1.087cm}{1.497cm}}{\pgfqpoint{1.072cm}{1.463cm}}{\pgfqpoint{1.072cm}{1.426cm}}
\pgfpathcurveto{\pgfqpoint{1.072cm}{1.39cm}}{\pgfqpoint{1.087cm}{1.355cm}}{\pgfqpoint{1.112cm}{1.329cm}}
\pgfpathcurveto{\pgfqpoint{1.138cm}{1.304cm}}{\pgfqpoint{1.172cm}{1.289cm}}{\pgfqpoint{1.209cm}{1.289cm}}
\pgfpathcurveto{\pgfqpoint{1.245cm}{1.289cm}}{\pgfqpoint{1.28cm}{1.304cm}}{\pgfqpoint{1.305cm}{1.329cm}}
\pgfpathcurveto{\pgfqpoint{1.331cm}{1.355cm}}{\pgfqpoint{1.345cm}{1.39cm}}{\pgfqpoint{1.345cm}{1.426cm}}
\pgfusepath{fill}
\begin{pgfscope}
\pgfsetdash{}{0cm}
\pgfsetlinewidth{0.818mm}
\pgfsetroundcap
\pgfsetmiterlimit{4.0}
\pgfpathmoveto{\pgfqpoint{0.682cm}{0.726cm}}
\pgfpathlineto{\pgfqpoint{0.682cm}{0.097cm}}
\pgfusepath{stroke}
\end{pgfscope}
\end{pgfscope}
\end{pgfscope}
\end{pgfscope}
\end{tikzpicture}}} \right] \]
  \[ + \tau^{- \frac{1 + \nu}{2}} \left[ \LL, \left( 3 \tau^{\frac{1 +
     \nu}{2}} (- X^{\!\resizebox{0.6em}{!}{
\begin{tikzpicture}
\pgfpathmoveto{\pgfqpoint{0cm}{-0.035cm}}
\pgfpathlineto{\pgfqpoint{1.376cm}{-0.035cm}}
\pgfpathlineto{\pgfqpoint{1.376cm}{1.552cm}}
\pgfpathlineto{\pgfqpoint{0cm}{1.552cm}}
\pgfpathclose
\pgfusepath{clip}
\begin{pgfscope}
\begin{pgfscope}
\pgfpathmoveto{\pgfqpoint{0cm}{-0.035cm}}
\pgfpathlineto{\pgfqpoint{1.376cm}{-0.035cm}}
\pgfpathlineto{\pgfqpoint{1.376cm}{1.552cm}}
\pgfpathlineto{\pgfqpoint{0cm}{1.552cm}}
\pgfpathclose
\pgfusepath{clip}
\begin{pgfscope}
\begin{pgfscope}
\pgfsetdash{}{0cm}
\pgfsetlinewidth{0.818mm}
\pgfsetroundcap
\pgfsetroundjoin
\pgfsetmiterlimit{7.0}
\definecolor{eps2pgf_color}{gray}{0}\pgfsetstrokecolor{eps2pgf_color}\pgfsetfillcolor{eps2pgf_color}
\pgfpathmoveto{\pgfqpoint{0.117cm}{1.421cm}}
\pgfpathlineto{\pgfqpoint{0.682cm}{0.671cm}}
\pgfpathlineto{\pgfqpoint{1.246cm}{1.421cm}}
\pgfusepath{stroke}
\end{pgfscope}
\definecolor{eps2pgf_color}{gray}{0}\pgfsetstrokecolor{eps2pgf_color}\pgfsetfillcolor{eps2pgf_color}
\pgfpathmoveto{\pgfqpoint{0.273cm}{1.395cm}}
\pgfpathcurveto{\pgfqpoint{0.273cm}{1.432cm}}{\pgfqpoint{0.259cm}{1.467cm}}{\pgfqpoint{0.233cm}{1.492cm}}
\pgfpathcurveto{\pgfqpoint{0.207cm}{1.518cm}}{\pgfqpoint{0.173cm}{1.532cm}}{\pgfqpoint{0.137cm}{1.532cm}}
\pgfpathcurveto{\pgfqpoint{0.1cm}{1.532cm}}{\pgfqpoint{0.066cm}{1.518cm}}{\pgfqpoint{0.04cm}{1.492cm}}
\pgfpathcurveto{\pgfqpoint{0.014cm}{1.467cm}}{\pgfqpoint{0cm}{1.432cm}}{\pgfqpoint{0cm}{1.395cm}}
\pgfpathcurveto{\pgfqpoint{0cm}{1.359cm}}{\pgfqpoint{0.014cm}{1.324cm}}{\pgfqpoint{0.04cm}{1.299cm}}
\pgfpathcurveto{\pgfqpoint{0.066cm}{1.273cm}}{\pgfqpoint{0.1cm}{1.258cm}}{\pgfqpoint{0.137cm}{1.258cm}}
\pgfpathcurveto{\pgfqpoint{0.173cm}{1.258cm}}{\pgfqpoint{0.207cm}{1.273cm}}{\pgfqpoint{0.233cm}{1.299cm}}
\pgfpathcurveto{\pgfqpoint{0.259cm}{1.324cm}}{\pgfqpoint{0.273cm}{1.359cm}}{\pgfqpoint{0.273cm}{1.395cm}}
\pgfusepath{fill}
\begin{pgfscope}
\pgfsetdash{}{0cm}
\pgfsetlinewidth{0.818mm}
\pgfsetmiterlimit{7.0}
\pgfpathmoveto{\pgfqpoint{0.682cm}{0.671cm}}
\pgfpathlineto{\pgfqpoint{0.679cm}{1.418cm}}
\pgfusepath{stroke}
\end{pgfscope}
\pgfpathmoveto{\pgfqpoint{0.815cm}{1.399cm}}
\pgfpathcurveto{\pgfqpoint{0.815cm}{1.435cm}}{\pgfqpoint{0.801cm}{1.47cm}}{\pgfqpoint{0.775cm}{1.496cm}}
\pgfpathcurveto{\pgfqpoint{0.75cm}{1.521cm}}{\pgfqpoint{0.715cm}{1.536cm}}{\pgfqpoint{0.679cm}{1.536cm}}
\pgfpathcurveto{\pgfqpoint{0.643cm}{1.536cm}}{\pgfqpoint{0.608cm}{1.521cm}}{\pgfqpoint{0.582cm}{1.496cm}}
\pgfpathcurveto{\pgfqpoint{0.557cm}{1.47cm}}{\pgfqpoint{0.542cm}{1.435cm}}{\pgfqpoint{0.542cm}{1.399cm}}
\pgfpathcurveto{\pgfqpoint{0.542cm}{1.363cm}}{\pgfqpoint{0.557cm}{1.328cm}}{\pgfqpoint{0.582cm}{1.302cm}}
\pgfpathcurveto{\pgfqpoint{0.608cm}{1.276cm}}{\pgfqpoint{0.643cm}{1.262cm}}{\pgfqpoint{0.679cm}{1.262cm}}
\pgfpathcurveto{\pgfqpoint{0.715cm}{1.262cm}}{\pgfqpoint{0.75cm}{1.276cm}}{\pgfqpoint{0.775cm}{1.302cm}}
\pgfpathcurveto{\pgfqpoint{0.801cm}{1.328cm}}{\pgfqpoint{0.815cm}{1.363cm}}{\pgfqpoint{0.815cm}{1.399cm}}
\pgfusepath{fill}
\pgfpathmoveto{\pgfqpoint{1.345cm}{1.371cm}}
\pgfpathcurveto{\pgfqpoint{1.345cm}{1.408cm}}{\pgfqpoint{1.331cm}{1.442cm}}{\pgfqpoint{1.305cm}{1.468cm}}
\pgfpathcurveto{\pgfqpoint{1.28cm}{1.494cm}}{\pgfqpoint{1.245cm}{1.508cm}}{\pgfqpoint{1.209cm}{1.508cm}}
\pgfpathcurveto{\pgfqpoint{1.172cm}{1.508cm}}{\pgfqpoint{1.138cm}{1.494cm}}{\pgfqpoint{1.112cm}{1.468cm}}
\pgfpathcurveto{\pgfqpoint{1.087cm}{1.442cm}}{\pgfqpoint{1.072cm}{1.408cm}}{\pgfqpoint{1.072cm}{1.371cm}}
\pgfpathcurveto{\pgfqpoint{1.072cm}{1.335cm}}{\pgfqpoint{1.087cm}{1.3cm}}{\pgfqpoint{1.112cm}{1.274cm}}
\pgfpathcurveto{\pgfqpoint{1.138cm}{1.249cm}}{\pgfqpoint{1.172cm}{1.234cm}}{\pgfqpoint{1.209cm}{1.234cm}}
\pgfpathcurveto{\pgfqpoint{1.245cm}{1.234cm}}{\pgfqpoint{1.28cm}{1.249cm}}{\pgfqpoint{1.305cm}{1.274cm}}
\pgfpathcurveto{\pgfqpoint{1.331cm}{1.3cm}}{\pgfqpoint{1.345cm}{1.335cm}}{\pgfqpoint{1.345cm}{1.371cm}}
\pgfusepath{fill}
\begin{pgfscope}
\pgfsetdash{}{0cm}
\pgfsetlinewidth{0.818mm}
\pgfsetroundcap
\pgfsetmiterlimit{4.0}
\pgfpathmoveto{\pgfqpoint{0.682cm}{0.671cm}}
\pgfpathlineto{\pgfqpoint{0.682cm}{0.042cm}}
\pgfusepath{stroke}
\end{pgfscope}
\end{pgfscope}
\end{pgfscope}
\end{pgfscope}
\end{tikzpicture}}} + \phi + \psi) \right) \precprec \right]
     X^{\!\resizebox{0.6em}{!}{
\begin{tikzpicture}
\pgfpathmoveto{\pgfqpoint{0cm}{0cm}}
\pgfpathlineto{\pgfqpoint{1.376cm}{0cm}}
\pgfpathlineto{\pgfqpoint{1.376cm}{1.588cm}}
\pgfpathlineto{\pgfqpoint{0cm}{1.588cm}}
\pgfpathclose
\pgfusepath{clip}
\begin{pgfscope}
\begin{pgfscope}
\pgfpathmoveto{\pgfqpoint{0cm}{0cm}}
\pgfpathlineto{\pgfqpoint{1.376cm}{0cm}}
\pgfpathlineto{\pgfqpoint{1.376cm}{1.588cm}}
\pgfpathlineto{\pgfqpoint{0cm}{1.588cm}}
\pgfpathclose
\pgfusepath{clip}
\begin{pgfscope}
\begin{pgfscope}
\definecolor{eps2pgf_color}{gray}{0.976471}\pgfsetstrokecolor{eps2pgf_color}\pgfsetfillcolor{eps2pgf_color}
\pgfpathmoveto{\pgfqpoint{0cm}{0cm}}
\pgfpathlineto{\pgfqpoint{1.376cm}{0cm}}
\pgfpathlineto{\pgfqpoint{1.376cm}{1.588cm}}
\pgfpathlineto{\pgfqpoint{0cm}{1.588cm}}
\pgfpathclose
\pgfusepath{fill}
\end{pgfscope}
\begin{pgfscope}
\pgfsetdash{}{0cm}
\pgfsetlinewidth{0.818mm}
\pgfsetroundcap
\pgfsetroundjoin
\pgfsetmiterlimit{7.0}
\definecolor{eps2pgf_color}{gray}{0}\pgfsetstrokecolor{eps2pgf_color}\pgfsetfillcolor{eps2pgf_color}
\pgfpathmoveto{\pgfqpoint{0.117cm}{1.476cm}}
\pgfpathlineto{\pgfqpoint{0.682cm}{0.726cm}}
\pgfpathlineto{\pgfqpoint{1.246cm}{1.476cm}}
\pgfusepath{stroke}
\end{pgfscope}
\definecolor{eps2pgf_color}{gray}{0}\pgfsetstrokecolor{eps2pgf_color}\pgfsetfillcolor{eps2pgf_color}
\pgfpathmoveto{\pgfqpoint{0.273cm}{1.451cm}}
\pgfpathcurveto{\pgfqpoint{0.273cm}{1.487cm}}{\pgfqpoint{0.259cm}{1.522cm}}{\pgfqpoint{0.233cm}{1.547cm}}
\pgfpathcurveto{\pgfqpoint{0.207cm}{1.573cm}}{\pgfqpoint{0.173cm}{1.588cm}}{\pgfqpoint{0.137cm}{1.588cm}}
\pgfpathcurveto{\pgfqpoint{0.1cm}{1.588cm}}{\pgfqpoint{0.066cm}{1.573cm}}{\pgfqpoint{0.04cm}{1.547cm}}
\pgfpathcurveto{\pgfqpoint{0.014cm}{1.522cm}}{\pgfqpoint{0cm}{1.487cm}}{\pgfqpoint{0cm}{1.451cm}}
\pgfpathcurveto{\pgfqpoint{0cm}{1.414cm}}{\pgfqpoint{0.014cm}{1.379cm}}{\pgfqpoint{0.04cm}{1.354cm}}
\pgfpathcurveto{\pgfqpoint{0.066cm}{1.328cm}}{\pgfqpoint{0.1cm}{1.314cm}}{\pgfqpoint{0.137cm}{1.314cm}}
\pgfpathcurveto{\pgfqpoint{0.173cm}{1.314cm}}{\pgfqpoint{0.207cm}{1.328cm}}{\pgfqpoint{0.233cm}{1.354cm}}
\pgfpathcurveto{\pgfqpoint{0.259cm}{1.379cm}}{\pgfqpoint{0.273cm}{1.414cm}}{\pgfqpoint{0.273cm}{1.451cm}}
\pgfusepath{fill}
\pgfpathmoveto{\pgfqpoint{1.345cm}{1.426cm}}
\pgfpathcurveto{\pgfqpoint{1.345cm}{1.463cm}}{\pgfqpoint{1.331cm}{1.497cm}}{\pgfqpoint{1.305cm}{1.523cm}}
\pgfpathcurveto{\pgfqpoint{1.28cm}{1.549cm}}{\pgfqpoint{1.245cm}{1.563cm}}{\pgfqpoint{1.209cm}{1.563cm}}
\pgfpathcurveto{\pgfqpoint{1.172cm}{1.563cm}}{\pgfqpoint{1.138cm}{1.549cm}}{\pgfqpoint{1.112cm}{1.523cm}}
\pgfpathcurveto{\pgfqpoint{1.087cm}{1.497cm}}{\pgfqpoint{1.072cm}{1.463cm}}{\pgfqpoint{1.072cm}{1.426cm}}
\pgfpathcurveto{\pgfqpoint{1.072cm}{1.39cm}}{\pgfqpoint{1.087cm}{1.355cm}}{\pgfqpoint{1.112cm}{1.329cm}}
\pgfpathcurveto{\pgfqpoint{1.138cm}{1.304cm}}{\pgfqpoint{1.172cm}{1.289cm}}{\pgfqpoint{1.209cm}{1.289cm}}
\pgfpathcurveto{\pgfqpoint{1.245cm}{1.289cm}}{\pgfqpoint{1.28cm}{1.304cm}}{\pgfqpoint{1.305cm}{1.329cm}}
\pgfpathcurveto{\pgfqpoint{1.331cm}{1.355cm}}{\pgfqpoint{1.345cm}{1.39cm}}{\pgfqpoint{1.345cm}{1.426cm}}
\pgfusepath{fill}
\begin{pgfscope}
\pgfsetdash{}{0cm}
\pgfsetlinewidth{0.818mm}
\pgfsetroundcap
\pgfsetmiterlimit{4.0}
\pgfpathmoveto{\pgfqpoint{0.682cm}{0.726cm}}
\pgfpathlineto{\pgfqpoint{0.682cm}{0.097cm}}
\pgfusepath{stroke}
\end{pgfscope}
\end{pgfscope}
\end{pgfscope}
\end{pgfscope}
\end{tikzpicture}}} \backassign \tmop{com}_{21} + \tmop{com}_{22} . \]
  The first term can be estimated in $\CC^{1 - \kappa}$ with a suitable
  weight. Hence Lemma \ref{lemma:tau-weighted-schauder} allows us to
  compensate for the blow up in $\tau$. More precisely, for this term we apply
  Lemma \ref{lemma:tau-weighted-schauder} with $\beta_i = 2 - \kappa - \alpha$
  to obtain (provided $4 - \kappa \geq 3 + \nu$)
  \[ \| \tau^{(2 - \kappa - \alpha) / 2} \tmop{com}_{21} \|_{C \CC^{1 -
     \kappa} (\tau^{(2 + \alpha) / 2} \rho_{}^{3 + \gamma'})^{}} \]
  \[ \lesssim \left\| \left( \tau^{\frac{1 + \nu}{2}} (- X^{\!\resizebox{0.6em}{!}{
\begin{tikzpicture}
\pgfpathmoveto{\pgfqpoint{0cm}{-0.035cm}}
\pgfpathlineto{\pgfqpoint{1.376cm}{-0.035cm}}
\pgfpathlineto{\pgfqpoint{1.376cm}{1.552cm}}
\pgfpathlineto{\pgfqpoint{0cm}{1.552cm}}
\pgfpathclose
\pgfusepath{clip}
\begin{pgfscope}
\begin{pgfscope}
\pgfpathmoveto{\pgfqpoint{0cm}{-0.035cm}}
\pgfpathlineto{\pgfqpoint{1.376cm}{-0.035cm}}
\pgfpathlineto{\pgfqpoint{1.376cm}{1.552cm}}
\pgfpathlineto{\pgfqpoint{0cm}{1.552cm}}
\pgfpathclose
\pgfusepath{clip}
\begin{pgfscope}
\begin{pgfscope}
\pgfsetdash{}{0cm}
\pgfsetlinewidth{0.818mm}
\pgfsetroundcap
\pgfsetroundjoin
\pgfsetmiterlimit{7.0}
\definecolor{eps2pgf_color}{gray}{0}\pgfsetstrokecolor{eps2pgf_color}\pgfsetfillcolor{eps2pgf_color}
\pgfpathmoveto{\pgfqpoint{0.117cm}{1.421cm}}
\pgfpathlineto{\pgfqpoint{0.682cm}{0.671cm}}
\pgfpathlineto{\pgfqpoint{1.246cm}{1.421cm}}
\pgfusepath{stroke}
\end{pgfscope}
\definecolor{eps2pgf_color}{gray}{0}\pgfsetstrokecolor{eps2pgf_color}\pgfsetfillcolor{eps2pgf_color}
\pgfpathmoveto{\pgfqpoint{0.273cm}{1.395cm}}
\pgfpathcurveto{\pgfqpoint{0.273cm}{1.432cm}}{\pgfqpoint{0.259cm}{1.467cm}}{\pgfqpoint{0.233cm}{1.492cm}}
\pgfpathcurveto{\pgfqpoint{0.207cm}{1.518cm}}{\pgfqpoint{0.173cm}{1.532cm}}{\pgfqpoint{0.137cm}{1.532cm}}
\pgfpathcurveto{\pgfqpoint{0.1cm}{1.532cm}}{\pgfqpoint{0.066cm}{1.518cm}}{\pgfqpoint{0.04cm}{1.492cm}}
\pgfpathcurveto{\pgfqpoint{0.014cm}{1.467cm}}{\pgfqpoint{0cm}{1.432cm}}{\pgfqpoint{0cm}{1.395cm}}
\pgfpathcurveto{\pgfqpoint{0cm}{1.359cm}}{\pgfqpoint{0.014cm}{1.324cm}}{\pgfqpoint{0.04cm}{1.299cm}}
\pgfpathcurveto{\pgfqpoint{0.066cm}{1.273cm}}{\pgfqpoint{0.1cm}{1.258cm}}{\pgfqpoint{0.137cm}{1.258cm}}
\pgfpathcurveto{\pgfqpoint{0.173cm}{1.258cm}}{\pgfqpoint{0.207cm}{1.273cm}}{\pgfqpoint{0.233cm}{1.299cm}}
\pgfpathcurveto{\pgfqpoint{0.259cm}{1.324cm}}{\pgfqpoint{0.273cm}{1.359cm}}{\pgfqpoint{0.273cm}{1.395cm}}
\pgfusepath{fill}
\begin{pgfscope}
\pgfsetdash{}{0cm}
\pgfsetlinewidth{0.818mm}
\pgfsetmiterlimit{7.0}
\pgfpathmoveto{\pgfqpoint{0.682cm}{0.671cm}}
\pgfpathlineto{\pgfqpoint{0.679cm}{1.418cm}}
\pgfusepath{stroke}
\end{pgfscope}
\pgfpathmoveto{\pgfqpoint{0.815cm}{1.399cm}}
\pgfpathcurveto{\pgfqpoint{0.815cm}{1.435cm}}{\pgfqpoint{0.801cm}{1.47cm}}{\pgfqpoint{0.775cm}{1.496cm}}
\pgfpathcurveto{\pgfqpoint{0.75cm}{1.521cm}}{\pgfqpoint{0.715cm}{1.536cm}}{\pgfqpoint{0.679cm}{1.536cm}}
\pgfpathcurveto{\pgfqpoint{0.643cm}{1.536cm}}{\pgfqpoint{0.608cm}{1.521cm}}{\pgfqpoint{0.582cm}{1.496cm}}
\pgfpathcurveto{\pgfqpoint{0.557cm}{1.47cm}}{\pgfqpoint{0.542cm}{1.435cm}}{\pgfqpoint{0.542cm}{1.399cm}}
\pgfpathcurveto{\pgfqpoint{0.542cm}{1.363cm}}{\pgfqpoint{0.557cm}{1.328cm}}{\pgfqpoint{0.582cm}{1.302cm}}
\pgfpathcurveto{\pgfqpoint{0.608cm}{1.276cm}}{\pgfqpoint{0.643cm}{1.262cm}}{\pgfqpoint{0.679cm}{1.262cm}}
\pgfpathcurveto{\pgfqpoint{0.715cm}{1.262cm}}{\pgfqpoint{0.75cm}{1.276cm}}{\pgfqpoint{0.775cm}{1.302cm}}
\pgfpathcurveto{\pgfqpoint{0.801cm}{1.328cm}}{\pgfqpoint{0.815cm}{1.363cm}}{\pgfqpoint{0.815cm}{1.399cm}}
\pgfusepath{fill}
\pgfpathmoveto{\pgfqpoint{1.345cm}{1.371cm}}
\pgfpathcurveto{\pgfqpoint{1.345cm}{1.408cm}}{\pgfqpoint{1.331cm}{1.442cm}}{\pgfqpoint{1.305cm}{1.468cm}}
\pgfpathcurveto{\pgfqpoint{1.28cm}{1.494cm}}{\pgfqpoint{1.245cm}{1.508cm}}{\pgfqpoint{1.209cm}{1.508cm}}
\pgfpathcurveto{\pgfqpoint{1.172cm}{1.508cm}}{\pgfqpoint{1.138cm}{1.494cm}}{\pgfqpoint{1.112cm}{1.468cm}}
\pgfpathcurveto{\pgfqpoint{1.087cm}{1.442cm}}{\pgfqpoint{1.072cm}{1.408cm}}{\pgfqpoint{1.072cm}{1.371cm}}
\pgfpathcurveto{\pgfqpoint{1.072cm}{1.335cm}}{\pgfqpoint{1.087cm}{1.3cm}}{\pgfqpoint{1.112cm}{1.274cm}}
\pgfpathcurveto{\pgfqpoint{1.138cm}{1.249cm}}{\pgfqpoint{1.172cm}{1.234cm}}{\pgfqpoint{1.209cm}{1.234cm}}
\pgfpathcurveto{\pgfqpoint{1.245cm}{1.234cm}}{\pgfqpoint{1.28cm}{1.249cm}}{\pgfqpoint{1.305cm}{1.274cm}}
\pgfpathcurveto{\pgfqpoint{1.331cm}{1.3cm}}{\pgfqpoint{1.345cm}{1.335cm}}{\pgfqpoint{1.345cm}{1.371cm}}
\pgfusepath{fill}
\begin{pgfscope}
\pgfsetdash{}{0cm}
\pgfsetlinewidth{0.818mm}
\pgfsetroundcap
\pgfsetmiterlimit{4.0}
\pgfpathmoveto{\pgfqpoint{0.682cm}{0.671cm}}
\pgfpathlineto{\pgfqpoint{0.682cm}{0.042cm}}
\pgfusepath{stroke}
\end{pgfscope}
\end{pgfscope}
\end{pgfscope}
\end{pgfscope}
\end{tikzpicture}}} + \phi
     + \psi) \right) \precprec X^{\!\resizebox{0.6em}{!}{
\begin{tikzpicture}
\pgfpathmoveto{\pgfqpoint{0cm}{0cm}}
\pgfpathlineto{\pgfqpoint{1.376cm}{0cm}}
\pgfpathlineto{\pgfqpoint{1.376cm}{1.588cm}}
\pgfpathlineto{\pgfqpoint{0cm}{1.588cm}}
\pgfpathclose
\pgfusepath{clip}
\begin{pgfscope}
\begin{pgfscope}
\pgfpathmoveto{\pgfqpoint{0cm}{0cm}}
\pgfpathlineto{\pgfqpoint{1.376cm}{0cm}}
\pgfpathlineto{\pgfqpoint{1.376cm}{1.588cm}}
\pgfpathlineto{\pgfqpoint{0cm}{1.588cm}}
\pgfpathclose
\pgfusepath{clip}
\begin{pgfscope}
\begin{pgfscope}
\definecolor{eps2pgf_color}{gray}{0.976471}\pgfsetstrokecolor{eps2pgf_color}\pgfsetfillcolor{eps2pgf_color}
\pgfpathmoveto{\pgfqpoint{0cm}{0cm}}
\pgfpathlineto{\pgfqpoint{1.376cm}{0cm}}
\pgfpathlineto{\pgfqpoint{1.376cm}{1.588cm}}
\pgfpathlineto{\pgfqpoint{0cm}{1.588cm}}
\pgfpathclose
\pgfusepath{fill}
\end{pgfscope}
\begin{pgfscope}
\pgfsetdash{}{0cm}
\pgfsetlinewidth{0.818mm}
\pgfsetroundcap
\pgfsetroundjoin
\pgfsetmiterlimit{7.0}
\definecolor{eps2pgf_color}{gray}{0}\pgfsetstrokecolor{eps2pgf_color}\pgfsetfillcolor{eps2pgf_color}
\pgfpathmoveto{\pgfqpoint{0.117cm}{1.476cm}}
\pgfpathlineto{\pgfqpoint{0.682cm}{0.726cm}}
\pgfpathlineto{\pgfqpoint{1.246cm}{1.476cm}}
\pgfusepath{stroke}
\end{pgfscope}
\definecolor{eps2pgf_color}{gray}{0}\pgfsetstrokecolor{eps2pgf_color}\pgfsetfillcolor{eps2pgf_color}
\pgfpathmoveto{\pgfqpoint{0.273cm}{1.451cm}}
\pgfpathcurveto{\pgfqpoint{0.273cm}{1.487cm}}{\pgfqpoint{0.259cm}{1.522cm}}{\pgfqpoint{0.233cm}{1.547cm}}
\pgfpathcurveto{\pgfqpoint{0.207cm}{1.573cm}}{\pgfqpoint{0.173cm}{1.588cm}}{\pgfqpoint{0.137cm}{1.588cm}}
\pgfpathcurveto{\pgfqpoint{0.1cm}{1.588cm}}{\pgfqpoint{0.066cm}{1.573cm}}{\pgfqpoint{0.04cm}{1.547cm}}
\pgfpathcurveto{\pgfqpoint{0.014cm}{1.522cm}}{\pgfqpoint{0cm}{1.487cm}}{\pgfqpoint{0cm}{1.451cm}}
\pgfpathcurveto{\pgfqpoint{0cm}{1.414cm}}{\pgfqpoint{0.014cm}{1.379cm}}{\pgfqpoint{0.04cm}{1.354cm}}
\pgfpathcurveto{\pgfqpoint{0.066cm}{1.328cm}}{\pgfqpoint{0.1cm}{1.314cm}}{\pgfqpoint{0.137cm}{1.314cm}}
\pgfpathcurveto{\pgfqpoint{0.173cm}{1.314cm}}{\pgfqpoint{0.207cm}{1.328cm}}{\pgfqpoint{0.233cm}{1.354cm}}
\pgfpathcurveto{\pgfqpoint{0.259cm}{1.379cm}}{\pgfqpoint{0.273cm}{1.414cm}}{\pgfqpoint{0.273cm}{1.451cm}}
\pgfusepath{fill}
\pgfpathmoveto{\pgfqpoint{1.345cm}{1.426cm}}
\pgfpathcurveto{\pgfqpoint{1.345cm}{1.463cm}}{\pgfqpoint{1.331cm}{1.497cm}}{\pgfqpoint{1.305cm}{1.523cm}}
\pgfpathcurveto{\pgfqpoint{1.28cm}{1.549cm}}{\pgfqpoint{1.245cm}{1.563cm}}{\pgfqpoint{1.209cm}{1.563cm}}
\pgfpathcurveto{\pgfqpoint{1.172cm}{1.563cm}}{\pgfqpoint{1.138cm}{1.549cm}}{\pgfqpoint{1.112cm}{1.523cm}}
\pgfpathcurveto{\pgfqpoint{1.087cm}{1.497cm}}{\pgfqpoint{1.072cm}{1.463cm}}{\pgfqpoint{1.072cm}{1.426cm}}
\pgfpathcurveto{\pgfqpoint{1.072cm}{1.39cm}}{\pgfqpoint{1.087cm}{1.355cm}}{\pgfqpoint{1.112cm}{1.329cm}}
\pgfpathcurveto{\pgfqpoint{1.138cm}{1.304cm}}{\pgfqpoint{1.172cm}{1.289cm}}{\pgfqpoint{1.209cm}{1.289cm}}
\pgfpathcurveto{\pgfqpoint{1.245cm}{1.289cm}}{\pgfqpoint{1.28cm}{1.304cm}}{\pgfqpoint{1.305cm}{1.329cm}}
\pgfpathcurveto{\pgfqpoint{1.331cm}{1.355cm}}{\pgfqpoint{1.345cm}{1.39cm}}{\pgfqpoint{1.345cm}{1.426cm}}
\pgfusepath{fill}
\begin{pgfscope}
\pgfsetdash{}{0cm}
\pgfsetlinewidth{0.818mm}
\pgfsetroundcap
\pgfsetmiterlimit{4.0}
\pgfpathmoveto{\pgfqpoint{0.682cm}{0.726cm}}
\pgfpathlineto{\pgfqpoint{0.682cm}{0.097cm}}
\pgfusepath{stroke}
\end{pgfscope}
\end{pgfscope}
\end{pgfscope}
\end{pgfscope}
\end{tikzpicture}}} \right\|_{C \CC^{1 - \kappa}
     (\rho_{}^{3 + \gamma'})} \lesssim 1 + \| \tau^{1 / 2} (\phi + \psi)
     \|_{L^{\infty} L^{\infty} (\rho)} . \]
  The second term can be estimated using Lemma \ref{lem:5.1} in $\CC^{- 1 +
  \alpha} (\rho^{3 + \gamma'})$. So we again apply Lemma~\ref{lemma:tau-weighted-schauder} with $\beta_i = 0$ to deduce (since $2 +
  \alpha \geq 1 + \nu$)
  \[ \| \tmop{com}_{22} \|_{C \CC^{- 1 + \alpha} (\tau^{(2 + \alpha) / 2}
     \rho_{}^{3 + \gamma'})} \]
  \[ \lesssim 1 + \left\| \tau^{\frac{1 + \nu}{2}} (\phi + \psi)
     \right\|_{C^{(\alpha + \kappa) / 2} L^{\infty} (\rho^{3 + \gamma''})} +
     \left\| \tau^{\frac{1 + \nu}{2}} (\phi + \psi) \right\|_{C^{} \CC^{\alpha
     + \kappa} (\rho^{3 + \gamma''})} . \]
  All the other terms in $\breve{\Theta}$ can be estimated as in Section
  \ref{sec:43}, or more precisely pointwise in time by the approach of Section \ref{ssec:theta}. To summarize, Lemma \ref{lemma:tau-weighted-schauder} gives
  \[ \| \breve{\vartheta} \|_{C \CC^{1 + \alpha} (\tau^{(2 + \alpha) / 2}
     \rho^{3+\gamma'})} \lesssim \| \breve{\vartheta} \|_{C \CC^{1 / 2 +
     \alpha} (\tau^{3 / 4 + \alpha / 2} \rho^{3 / 2 + \alpha})} + 1 \]
  \[ + \| \psi \|_{C \CC^{1 / 2 + \alpha} (\tau^{3 / 4 + \alpha / 2} \rho^{3 /
     2 + \alpha})} + \| \psi \|_{L^{\infty} L^{\infty} (\tau^{1 / 2} \rho)}^{1
     + \varepsilon} \]
  \[ + \| \tau^{\frac{1 + \nu}{2}} (\phi + \psi) \|_{C^{(\alpha + \kappa)
     / 2} L^{\infty} (\rho^{3 + \gamma''})} + \left\| \tau^{\frac{1 + \nu}{2}}
     (\phi + \psi) \right\|_{C^{} \CC^{\alpha + \kappa} (\rho^{3 +
     \gamma''})}, \]
  where the only term requiring some time regularity is the second last one. This
  is the reason we introduced the modified $\vartheta$. In order to estimate
  the first term on the right hand side, we use the definition of
  $\breve{\vartheta}$ in \eqref{eq:th2} together with \eqref{eq:ph-down} to obtain (provided $3/2 + \alpha \geq 1+ \nu$)
  \[ \| \breve{\vartheta} \|_{C \CC^{1 / 2 + \alpha} (\tau^{3 / 4 + \alpha /
     2} \rho^{3 / 2 + \alpha})} \leqslant \| \phi \|_{C \CC^{1 / 2 + \alpha}
     (\tau^{3 / 4 + \alpha / 2} \rho^{3 / 2 + \alpha})} \]
  \[ + \left\| \tau^{- \frac{1 + \nu}{2}} (3 [\tau^{\frac{1 + \nu}{2}} (-
     X^{\!\resizebox{0.6em}{!}{
\begin{tikzpicture}
\pgfpathmoveto{\pgfqpoint{0cm}{-0.035cm}}
\pgfpathlineto{\pgfqpoint{1.376cm}{-0.035cm}}
\pgfpathlineto{\pgfqpoint{1.376cm}{1.552cm}}
\pgfpathlineto{\pgfqpoint{0cm}{1.552cm}}
\pgfpathclose
\pgfusepath{clip}
\begin{pgfscope}
\begin{pgfscope}
\pgfpathmoveto{\pgfqpoint{0cm}{-0.035cm}}
\pgfpathlineto{\pgfqpoint{1.376cm}{-0.035cm}}
\pgfpathlineto{\pgfqpoint{1.376cm}{1.552cm}}
\pgfpathlineto{\pgfqpoint{0cm}{1.552cm}}
\pgfpathclose
\pgfusepath{clip}
\begin{pgfscope}
\begin{pgfscope}
\pgfsetdash{}{0cm}
\pgfsetlinewidth{0.818mm}
\pgfsetroundcap
\pgfsetroundjoin
\pgfsetmiterlimit{7.0}
\definecolor{eps2pgf_color}{gray}{0}\pgfsetstrokecolor{eps2pgf_color}\pgfsetfillcolor{eps2pgf_color}
\pgfpathmoveto{\pgfqpoint{0.117cm}{1.421cm}}
\pgfpathlineto{\pgfqpoint{0.682cm}{0.671cm}}
\pgfpathlineto{\pgfqpoint{1.246cm}{1.421cm}}
\pgfusepath{stroke}
\end{pgfscope}
\definecolor{eps2pgf_color}{gray}{0}\pgfsetstrokecolor{eps2pgf_color}\pgfsetfillcolor{eps2pgf_color}
\pgfpathmoveto{\pgfqpoint{0.273cm}{1.395cm}}
\pgfpathcurveto{\pgfqpoint{0.273cm}{1.432cm}}{\pgfqpoint{0.259cm}{1.467cm}}{\pgfqpoint{0.233cm}{1.492cm}}
\pgfpathcurveto{\pgfqpoint{0.207cm}{1.518cm}}{\pgfqpoint{0.173cm}{1.532cm}}{\pgfqpoint{0.137cm}{1.532cm}}
\pgfpathcurveto{\pgfqpoint{0.1cm}{1.532cm}}{\pgfqpoint{0.066cm}{1.518cm}}{\pgfqpoint{0.04cm}{1.492cm}}
\pgfpathcurveto{\pgfqpoint{0.014cm}{1.467cm}}{\pgfqpoint{0cm}{1.432cm}}{\pgfqpoint{0cm}{1.395cm}}
\pgfpathcurveto{\pgfqpoint{0cm}{1.359cm}}{\pgfqpoint{0.014cm}{1.324cm}}{\pgfqpoint{0.04cm}{1.299cm}}
\pgfpathcurveto{\pgfqpoint{0.066cm}{1.273cm}}{\pgfqpoint{0.1cm}{1.258cm}}{\pgfqpoint{0.137cm}{1.258cm}}
\pgfpathcurveto{\pgfqpoint{0.173cm}{1.258cm}}{\pgfqpoint{0.207cm}{1.273cm}}{\pgfqpoint{0.233cm}{1.299cm}}
\pgfpathcurveto{\pgfqpoint{0.259cm}{1.324cm}}{\pgfqpoint{0.273cm}{1.359cm}}{\pgfqpoint{0.273cm}{1.395cm}}
\pgfusepath{fill}
\begin{pgfscope}
\pgfsetdash{}{0cm}
\pgfsetlinewidth{0.818mm}
\pgfsetmiterlimit{7.0}
\pgfpathmoveto{\pgfqpoint{0.682cm}{0.671cm}}
\pgfpathlineto{\pgfqpoint{0.679cm}{1.418cm}}
\pgfusepath{stroke}
\end{pgfscope}
\pgfpathmoveto{\pgfqpoint{0.815cm}{1.399cm}}
\pgfpathcurveto{\pgfqpoint{0.815cm}{1.435cm}}{\pgfqpoint{0.801cm}{1.47cm}}{\pgfqpoint{0.775cm}{1.496cm}}
\pgfpathcurveto{\pgfqpoint{0.75cm}{1.521cm}}{\pgfqpoint{0.715cm}{1.536cm}}{\pgfqpoint{0.679cm}{1.536cm}}
\pgfpathcurveto{\pgfqpoint{0.643cm}{1.536cm}}{\pgfqpoint{0.608cm}{1.521cm}}{\pgfqpoint{0.582cm}{1.496cm}}
\pgfpathcurveto{\pgfqpoint{0.557cm}{1.47cm}}{\pgfqpoint{0.542cm}{1.435cm}}{\pgfqpoint{0.542cm}{1.399cm}}
\pgfpathcurveto{\pgfqpoint{0.542cm}{1.363cm}}{\pgfqpoint{0.557cm}{1.328cm}}{\pgfqpoint{0.582cm}{1.302cm}}
\pgfpathcurveto{\pgfqpoint{0.608cm}{1.276cm}}{\pgfqpoint{0.643cm}{1.262cm}}{\pgfqpoint{0.679cm}{1.262cm}}
\pgfpathcurveto{\pgfqpoint{0.715cm}{1.262cm}}{\pgfqpoint{0.75cm}{1.276cm}}{\pgfqpoint{0.775cm}{1.302cm}}
\pgfpathcurveto{\pgfqpoint{0.801cm}{1.328cm}}{\pgfqpoint{0.815cm}{1.363cm}}{\pgfqpoint{0.815cm}{1.399cm}}
\pgfusepath{fill}
\pgfpathmoveto{\pgfqpoint{1.345cm}{1.371cm}}
\pgfpathcurveto{\pgfqpoint{1.345cm}{1.408cm}}{\pgfqpoint{1.331cm}{1.442cm}}{\pgfqpoint{1.305cm}{1.468cm}}
\pgfpathcurveto{\pgfqpoint{1.28cm}{1.494cm}}{\pgfqpoint{1.245cm}{1.508cm}}{\pgfqpoint{1.209cm}{1.508cm}}
\pgfpathcurveto{\pgfqpoint{1.172cm}{1.508cm}}{\pgfqpoint{1.138cm}{1.494cm}}{\pgfqpoint{1.112cm}{1.468cm}}
\pgfpathcurveto{\pgfqpoint{1.087cm}{1.442cm}}{\pgfqpoint{1.072cm}{1.408cm}}{\pgfqpoint{1.072cm}{1.371cm}}
\pgfpathcurveto{\pgfqpoint{1.072cm}{1.335cm}}{\pgfqpoint{1.087cm}{1.3cm}}{\pgfqpoint{1.112cm}{1.274cm}}
\pgfpathcurveto{\pgfqpoint{1.138cm}{1.249cm}}{\pgfqpoint{1.172cm}{1.234cm}}{\pgfqpoint{1.209cm}{1.234cm}}
\pgfpathcurveto{\pgfqpoint{1.245cm}{1.234cm}}{\pgfqpoint{1.28cm}{1.249cm}}{\pgfqpoint{1.305cm}{1.274cm}}
\pgfpathcurveto{\pgfqpoint{1.331cm}{1.3cm}}{\pgfqpoint{1.345cm}{1.335cm}}{\pgfqpoint{1.345cm}{1.371cm}}
\pgfusepath{fill}
\begin{pgfscope}
\pgfsetdash{}{0cm}
\pgfsetlinewidth{0.818mm}
\pgfsetroundcap
\pgfsetmiterlimit{4.0}
\pgfpathmoveto{\pgfqpoint{0.682cm}{0.671cm}}
\pgfpathlineto{\pgfqpoint{0.682cm}{0.042cm}}
\pgfusepath{stroke}
\end{pgfscope}
\end{pgfscope}
\end{pgfscope}
\end{pgfscope}
\end{tikzpicture}}} + \phi + \psi)] \precprec X^{\!\resizebox{0.6em}{!}{
\begin{tikzpicture}
\pgfpathmoveto{\pgfqpoint{0cm}{0cm}}
\pgfpathlineto{\pgfqpoint{1.376cm}{0cm}}
\pgfpathlineto{\pgfqpoint{1.376cm}{1.588cm}}
\pgfpathlineto{\pgfqpoint{0cm}{1.588cm}}
\pgfpathclose
\pgfusepath{clip}
\begin{pgfscope}
\begin{pgfscope}
\pgfpathmoveto{\pgfqpoint{0cm}{0cm}}
\pgfpathlineto{\pgfqpoint{1.376cm}{0cm}}
\pgfpathlineto{\pgfqpoint{1.376cm}{1.588cm}}
\pgfpathlineto{\pgfqpoint{0cm}{1.588cm}}
\pgfpathclose
\pgfusepath{clip}
\begin{pgfscope}
\begin{pgfscope}
\definecolor{eps2pgf_color}{gray}{0.976471}\pgfsetstrokecolor{eps2pgf_color}\pgfsetfillcolor{eps2pgf_color}
\pgfpathmoveto{\pgfqpoint{0cm}{0cm}}
\pgfpathlineto{\pgfqpoint{1.376cm}{0cm}}
\pgfpathlineto{\pgfqpoint{1.376cm}{1.588cm}}
\pgfpathlineto{\pgfqpoint{0cm}{1.588cm}}
\pgfpathclose
\pgfusepath{fill}
\end{pgfscope}
\begin{pgfscope}
\pgfsetdash{}{0cm}
\pgfsetlinewidth{0.818mm}
\pgfsetroundcap
\pgfsetroundjoin
\pgfsetmiterlimit{7.0}
\definecolor{eps2pgf_color}{gray}{0}\pgfsetstrokecolor{eps2pgf_color}\pgfsetfillcolor{eps2pgf_color}
\pgfpathmoveto{\pgfqpoint{0.117cm}{1.476cm}}
\pgfpathlineto{\pgfqpoint{0.682cm}{0.726cm}}
\pgfpathlineto{\pgfqpoint{1.246cm}{1.476cm}}
\pgfusepath{stroke}
\end{pgfscope}
\definecolor{eps2pgf_color}{gray}{0}\pgfsetstrokecolor{eps2pgf_color}\pgfsetfillcolor{eps2pgf_color}
\pgfpathmoveto{\pgfqpoint{0.273cm}{1.451cm}}
\pgfpathcurveto{\pgfqpoint{0.273cm}{1.487cm}}{\pgfqpoint{0.259cm}{1.522cm}}{\pgfqpoint{0.233cm}{1.547cm}}
\pgfpathcurveto{\pgfqpoint{0.207cm}{1.573cm}}{\pgfqpoint{0.173cm}{1.588cm}}{\pgfqpoint{0.137cm}{1.588cm}}
\pgfpathcurveto{\pgfqpoint{0.1cm}{1.588cm}}{\pgfqpoint{0.066cm}{1.573cm}}{\pgfqpoint{0.04cm}{1.547cm}}
\pgfpathcurveto{\pgfqpoint{0.014cm}{1.522cm}}{\pgfqpoint{0cm}{1.487cm}}{\pgfqpoint{0cm}{1.451cm}}
\pgfpathcurveto{\pgfqpoint{0cm}{1.414cm}}{\pgfqpoint{0.014cm}{1.379cm}}{\pgfqpoint{0.04cm}{1.354cm}}
\pgfpathcurveto{\pgfqpoint{0.066cm}{1.328cm}}{\pgfqpoint{0.1cm}{1.314cm}}{\pgfqpoint{0.137cm}{1.314cm}}
\pgfpathcurveto{\pgfqpoint{0.173cm}{1.314cm}}{\pgfqpoint{0.207cm}{1.328cm}}{\pgfqpoint{0.233cm}{1.354cm}}
\pgfpathcurveto{\pgfqpoint{0.259cm}{1.379cm}}{\pgfqpoint{0.273cm}{1.414cm}}{\pgfqpoint{0.273cm}{1.451cm}}
\pgfusepath{fill}
\pgfpathmoveto{\pgfqpoint{1.345cm}{1.426cm}}
\pgfpathcurveto{\pgfqpoint{1.345cm}{1.463cm}}{\pgfqpoint{1.331cm}{1.497cm}}{\pgfqpoint{1.305cm}{1.523cm}}
\pgfpathcurveto{\pgfqpoint{1.28cm}{1.549cm}}{\pgfqpoint{1.245cm}{1.563cm}}{\pgfqpoint{1.209cm}{1.563cm}}
\pgfpathcurveto{\pgfqpoint{1.172cm}{1.563cm}}{\pgfqpoint{1.138cm}{1.549cm}}{\pgfqpoint{1.112cm}{1.523cm}}
\pgfpathcurveto{\pgfqpoint{1.087cm}{1.497cm}}{\pgfqpoint{1.072cm}{1.463cm}}{\pgfqpoint{1.072cm}{1.426cm}}
\pgfpathcurveto{\pgfqpoint{1.072cm}{1.39cm}}{\pgfqpoint{1.087cm}{1.355cm}}{\pgfqpoint{1.112cm}{1.329cm}}
\pgfpathcurveto{\pgfqpoint{1.138cm}{1.304cm}}{\pgfqpoint{1.172cm}{1.289cm}}{\pgfqpoint{1.209cm}{1.289cm}}
\pgfpathcurveto{\pgfqpoint{1.245cm}{1.289cm}}{\pgfqpoint{1.28cm}{1.304cm}}{\pgfqpoint{1.305cm}{1.329cm}}
\pgfpathcurveto{\pgfqpoint{1.331cm}{1.355cm}}{\pgfqpoint{1.345cm}{1.39cm}}{\pgfqpoint{1.345cm}{1.426cm}}
\pgfusepath{fill}
\begin{pgfscope}
\pgfsetdash{}{0cm}
\pgfsetlinewidth{0.818mm}
\pgfsetroundcap
\pgfsetmiterlimit{4.0}
\pgfpathmoveto{\pgfqpoint{0.682cm}{0.726cm}}
\pgfpathlineto{\pgfqpoint{0.682cm}{0.097cm}}
\pgfusepath{stroke}
\end{pgfscope}
\end{pgfscope}
\end{pgfscope}
\end{pgfscope}
\end{tikzpicture}}}) \right\|_{C \CC^{1 /
     2 + \alpha} (\tau^{3 / 4 + \alpha / 2} \rho^{3 / 2 + \alpha})} \]
  \[ \lesssim 1 + \| \psi \|_{L^{\infty} L^{\infty} (\tau^{1 / 2}
     \rho)}^{\varepsilon} + \left\| \tau^{- \frac{1+\nu}{2}} ([\tau^{\frac{1 +
     \nu}{2}} (- X^{\!\resizebox{0.6em}{!}{
\begin{tikzpicture}
\pgfpathmoveto{\pgfqpoint{0cm}{-0.035cm}}
\pgfpathlineto{\pgfqpoint{1.376cm}{-0.035cm}}
\pgfpathlineto{\pgfqpoint{1.376cm}{1.552cm}}
\pgfpathlineto{\pgfqpoint{0cm}{1.552cm}}
\pgfpathclose
\pgfusepath{clip}
\begin{pgfscope}
\begin{pgfscope}
\pgfpathmoveto{\pgfqpoint{0cm}{-0.035cm}}
\pgfpathlineto{\pgfqpoint{1.376cm}{-0.035cm}}
\pgfpathlineto{\pgfqpoint{1.376cm}{1.552cm}}
\pgfpathlineto{\pgfqpoint{0cm}{1.552cm}}
\pgfpathclose
\pgfusepath{clip}
\begin{pgfscope}
\begin{pgfscope}
\pgfsetdash{}{0cm}
\pgfsetlinewidth{0.818mm}
\pgfsetroundcap
\pgfsetroundjoin
\pgfsetmiterlimit{7.0}
\definecolor{eps2pgf_color}{gray}{0}\pgfsetstrokecolor{eps2pgf_color}\pgfsetfillcolor{eps2pgf_color}
\pgfpathmoveto{\pgfqpoint{0.117cm}{1.421cm}}
\pgfpathlineto{\pgfqpoint{0.682cm}{0.671cm}}
\pgfpathlineto{\pgfqpoint{1.246cm}{1.421cm}}
\pgfusepath{stroke}
\end{pgfscope}
\definecolor{eps2pgf_color}{gray}{0}\pgfsetstrokecolor{eps2pgf_color}\pgfsetfillcolor{eps2pgf_color}
\pgfpathmoveto{\pgfqpoint{0.273cm}{1.395cm}}
\pgfpathcurveto{\pgfqpoint{0.273cm}{1.432cm}}{\pgfqpoint{0.259cm}{1.467cm}}{\pgfqpoint{0.233cm}{1.492cm}}
\pgfpathcurveto{\pgfqpoint{0.207cm}{1.518cm}}{\pgfqpoint{0.173cm}{1.532cm}}{\pgfqpoint{0.137cm}{1.532cm}}
\pgfpathcurveto{\pgfqpoint{0.1cm}{1.532cm}}{\pgfqpoint{0.066cm}{1.518cm}}{\pgfqpoint{0.04cm}{1.492cm}}
\pgfpathcurveto{\pgfqpoint{0.014cm}{1.467cm}}{\pgfqpoint{0cm}{1.432cm}}{\pgfqpoint{0cm}{1.395cm}}
\pgfpathcurveto{\pgfqpoint{0cm}{1.359cm}}{\pgfqpoint{0.014cm}{1.324cm}}{\pgfqpoint{0.04cm}{1.299cm}}
\pgfpathcurveto{\pgfqpoint{0.066cm}{1.273cm}}{\pgfqpoint{0.1cm}{1.258cm}}{\pgfqpoint{0.137cm}{1.258cm}}
\pgfpathcurveto{\pgfqpoint{0.173cm}{1.258cm}}{\pgfqpoint{0.207cm}{1.273cm}}{\pgfqpoint{0.233cm}{1.299cm}}
\pgfpathcurveto{\pgfqpoint{0.259cm}{1.324cm}}{\pgfqpoint{0.273cm}{1.359cm}}{\pgfqpoint{0.273cm}{1.395cm}}
\pgfusepath{fill}
\begin{pgfscope}
\pgfsetdash{}{0cm}
\pgfsetlinewidth{0.818mm}
\pgfsetmiterlimit{7.0}
\pgfpathmoveto{\pgfqpoint{0.682cm}{0.671cm}}
\pgfpathlineto{\pgfqpoint{0.679cm}{1.418cm}}
\pgfusepath{stroke}
\end{pgfscope}
\pgfpathmoveto{\pgfqpoint{0.815cm}{1.399cm}}
\pgfpathcurveto{\pgfqpoint{0.815cm}{1.435cm}}{\pgfqpoint{0.801cm}{1.47cm}}{\pgfqpoint{0.775cm}{1.496cm}}
\pgfpathcurveto{\pgfqpoint{0.75cm}{1.521cm}}{\pgfqpoint{0.715cm}{1.536cm}}{\pgfqpoint{0.679cm}{1.536cm}}
\pgfpathcurveto{\pgfqpoint{0.643cm}{1.536cm}}{\pgfqpoint{0.608cm}{1.521cm}}{\pgfqpoint{0.582cm}{1.496cm}}
\pgfpathcurveto{\pgfqpoint{0.557cm}{1.47cm}}{\pgfqpoint{0.542cm}{1.435cm}}{\pgfqpoint{0.542cm}{1.399cm}}
\pgfpathcurveto{\pgfqpoint{0.542cm}{1.363cm}}{\pgfqpoint{0.557cm}{1.328cm}}{\pgfqpoint{0.582cm}{1.302cm}}
\pgfpathcurveto{\pgfqpoint{0.608cm}{1.276cm}}{\pgfqpoint{0.643cm}{1.262cm}}{\pgfqpoint{0.679cm}{1.262cm}}
\pgfpathcurveto{\pgfqpoint{0.715cm}{1.262cm}}{\pgfqpoint{0.75cm}{1.276cm}}{\pgfqpoint{0.775cm}{1.302cm}}
\pgfpathcurveto{\pgfqpoint{0.801cm}{1.328cm}}{\pgfqpoint{0.815cm}{1.363cm}}{\pgfqpoint{0.815cm}{1.399cm}}
\pgfusepath{fill}
\pgfpathmoveto{\pgfqpoint{1.345cm}{1.371cm}}
\pgfpathcurveto{\pgfqpoint{1.345cm}{1.408cm}}{\pgfqpoint{1.331cm}{1.442cm}}{\pgfqpoint{1.305cm}{1.468cm}}
\pgfpathcurveto{\pgfqpoint{1.28cm}{1.494cm}}{\pgfqpoint{1.245cm}{1.508cm}}{\pgfqpoint{1.209cm}{1.508cm}}
\pgfpathcurveto{\pgfqpoint{1.172cm}{1.508cm}}{\pgfqpoint{1.138cm}{1.494cm}}{\pgfqpoint{1.112cm}{1.468cm}}
\pgfpathcurveto{\pgfqpoint{1.087cm}{1.442cm}}{\pgfqpoint{1.072cm}{1.408cm}}{\pgfqpoint{1.072cm}{1.371cm}}
\pgfpathcurveto{\pgfqpoint{1.072cm}{1.335cm}}{\pgfqpoint{1.087cm}{1.3cm}}{\pgfqpoint{1.112cm}{1.274cm}}
\pgfpathcurveto{\pgfqpoint{1.138cm}{1.249cm}}{\pgfqpoint{1.172cm}{1.234cm}}{\pgfqpoint{1.209cm}{1.234cm}}
\pgfpathcurveto{\pgfqpoint{1.245cm}{1.234cm}}{\pgfqpoint{1.28cm}{1.249cm}}{\pgfqpoint{1.305cm}{1.274cm}}
\pgfpathcurveto{\pgfqpoint{1.331cm}{1.3cm}}{\pgfqpoint{1.345cm}{1.335cm}}{\pgfqpoint{1.345cm}{1.371cm}}
\pgfusepath{fill}
\begin{pgfscope}
\pgfsetdash{}{0cm}
\pgfsetlinewidth{0.818mm}
\pgfsetroundcap
\pgfsetmiterlimit{4.0}
\pgfpathmoveto{\pgfqpoint{0.682cm}{0.671cm}}
\pgfpathlineto{\pgfqpoint{0.682cm}{0.042cm}}
\pgfusepath{stroke}
\end{pgfscope}
\end{pgfscope}
\end{pgfscope}
\end{pgfscope}
\end{tikzpicture}}} + \phi + \psi)] \precprec X^{\!\resizebox{0.6em}{!}{
\begin{tikzpicture}
\pgfpathmoveto{\pgfqpoint{0cm}{0cm}}
\pgfpathlineto{\pgfqpoint{1.376cm}{0cm}}
\pgfpathlineto{\pgfqpoint{1.376cm}{1.588cm}}
\pgfpathlineto{\pgfqpoint{0cm}{1.588cm}}
\pgfpathclose
\pgfusepath{clip}
\begin{pgfscope}
\begin{pgfscope}
\pgfpathmoveto{\pgfqpoint{0cm}{0cm}}
\pgfpathlineto{\pgfqpoint{1.376cm}{0cm}}
\pgfpathlineto{\pgfqpoint{1.376cm}{1.588cm}}
\pgfpathlineto{\pgfqpoint{0cm}{1.588cm}}
\pgfpathclose
\pgfusepath{clip}
\begin{pgfscope}
\begin{pgfscope}
\definecolor{eps2pgf_color}{gray}{0.976471}\pgfsetstrokecolor{eps2pgf_color}\pgfsetfillcolor{eps2pgf_color}
\pgfpathmoveto{\pgfqpoint{0cm}{0cm}}
\pgfpathlineto{\pgfqpoint{1.376cm}{0cm}}
\pgfpathlineto{\pgfqpoint{1.376cm}{1.588cm}}
\pgfpathlineto{\pgfqpoint{0cm}{1.588cm}}
\pgfpathclose
\pgfusepath{fill}
\end{pgfscope}
\begin{pgfscope}
\pgfsetdash{}{0cm}
\pgfsetlinewidth{0.818mm}
\pgfsetroundcap
\pgfsetroundjoin
\pgfsetmiterlimit{7.0}
\definecolor{eps2pgf_color}{gray}{0}\pgfsetstrokecolor{eps2pgf_color}\pgfsetfillcolor{eps2pgf_color}
\pgfpathmoveto{\pgfqpoint{0.117cm}{1.476cm}}
\pgfpathlineto{\pgfqpoint{0.682cm}{0.726cm}}
\pgfpathlineto{\pgfqpoint{1.246cm}{1.476cm}}
\pgfusepath{stroke}
\end{pgfscope}
\definecolor{eps2pgf_color}{gray}{0}\pgfsetstrokecolor{eps2pgf_color}\pgfsetfillcolor{eps2pgf_color}
\pgfpathmoveto{\pgfqpoint{0.273cm}{1.451cm}}
\pgfpathcurveto{\pgfqpoint{0.273cm}{1.487cm}}{\pgfqpoint{0.259cm}{1.522cm}}{\pgfqpoint{0.233cm}{1.547cm}}
\pgfpathcurveto{\pgfqpoint{0.207cm}{1.573cm}}{\pgfqpoint{0.173cm}{1.588cm}}{\pgfqpoint{0.137cm}{1.588cm}}
\pgfpathcurveto{\pgfqpoint{0.1cm}{1.588cm}}{\pgfqpoint{0.066cm}{1.573cm}}{\pgfqpoint{0.04cm}{1.547cm}}
\pgfpathcurveto{\pgfqpoint{0.014cm}{1.522cm}}{\pgfqpoint{0cm}{1.487cm}}{\pgfqpoint{0cm}{1.451cm}}
\pgfpathcurveto{\pgfqpoint{0cm}{1.414cm}}{\pgfqpoint{0.014cm}{1.379cm}}{\pgfqpoint{0.04cm}{1.354cm}}
\pgfpathcurveto{\pgfqpoint{0.066cm}{1.328cm}}{\pgfqpoint{0.1cm}{1.314cm}}{\pgfqpoint{0.137cm}{1.314cm}}
\pgfpathcurveto{\pgfqpoint{0.173cm}{1.314cm}}{\pgfqpoint{0.207cm}{1.328cm}}{\pgfqpoint{0.233cm}{1.354cm}}
\pgfpathcurveto{\pgfqpoint{0.259cm}{1.379cm}}{\pgfqpoint{0.273cm}{1.414cm}}{\pgfqpoint{0.273cm}{1.451cm}}
\pgfusepath{fill}
\pgfpathmoveto{\pgfqpoint{1.345cm}{1.426cm}}
\pgfpathcurveto{\pgfqpoint{1.345cm}{1.463cm}}{\pgfqpoint{1.331cm}{1.497cm}}{\pgfqpoint{1.305cm}{1.523cm}}
\pgfpathcurveto{\pgfqpoint{1.28cm}{1.549cm}}{\pgfqpoint{1.245cm}{1.563cm}}{\pgfqpoint{1.209cm}{1.563cm}}
\pgfpathcurveto{\pgfqpoint{1.172cm}{1.563cm}}{\pgfqpoint{1.138cm}{1.549cm}}{\pgfqpoint{1.112cm}{1.523cm}}
\pgfpathcurveto{\pgfqpoint{1.087cm}{1.497cm}}{\pgfqpoint{1.072cm}{1.463cm}}{\pgfqpoint{1.072cm}{1.426cm}}
\pgfpathcurveto{\pgfqpoint{1.072cm}{1.39cm}}{\pgfqpoint{1.087cm}{1.355cm}}{\pgfqpoint{1.112cm}{1.329cm}}
\pgfpathcurveto{\pgfqpoint{1.138cm}{1.304cm}}{\pgfqpoint{1.172cm}{1.289cm}}{\pgfqpoint{1.209cm}{1.289cm}}
\pgfpathcurveto{\pgfqpoint{1.245cm}{1.289cm}}{\pgfqpoint{1.28cm}{1.304cm}}{\pgfqpoint{1.305cm}{1.329cm}}
\pgfpathcurveto{\pgfqpoint{1.331cm}{1.355cm}}{\pgfqpoint{1.345cm}{1.39cm}}{\pgfqpoint{1.345cm}{1.426cm}}
\pgfusepath{fill}
\begin{pgfscope}
\pgfsetdash{}{0cm}
\pgfsetlinewidth{0.818mm}
\pgfsetroundcap
\pgfsetmiterlimit{4.0}
\pgfpathmoveto{\pgfqpoint{0.682cm}{0.726cm}}
\pgfpathlineto{\pgfqpoint{0.682cm}{0.097cm}}
\pgfusepath{stroke}
\end{pgfscope}
\end{pgfscope}
\end{pgfscope}
\end{pgfscope}
\end{tikzpicture}}})
     \right\|_{C \CC^{1 / 2 + \alpha} (\tau^{3 / 4 + \alpha / 2} \rho^{3 / 2 +
     \alpha})} \]
  \[ \lesssim 1 + \| \psi \|_{L^{\infty} L^{\infty} (\tau^{1 / 2}
     \rho)}^{\varepsilon} + \| \phi + \psi \|_{L^{\infty} L^{\infty} (\tau^{1
     / 2} \rho)} \lesssim 1 + \| \psi \|_{L^{\infty} L^{\infty}
     (\tau^{1 / 2} \rho)}^{} . \]
  Finally, it only remains to establish the time regularity of $\tau^{(1 +
  \nu) / 2} (\phi + \psi)$. A time-interpolation as in Lemma \ref{lemma:interp2} together with the Schauder estimates from
  Lemma~\ref{lemma:time-tau-schauder} (choosing the right regularity for each contributions and gaining
  powers of $\tau$)  yields for  $\delta\in(0,1)$
  \[ \| \tau^{\frac{1 + \nu}{2}} (\phi + \psi) \|_{C^{(\alpha + \kappa) / 2}
     L^{\infty} (\rho^{3 + \gamma''})}\lesssim  \| \tau^{\frac{1 + \nu}{2}} (\phi + \psi) \|_{L^{\8}     L^{\infty} (\rho)}+\delta \| \tau^{\frac{1 + \nu}{2}} (\phi + \psi) \|_{C^{(1/2+\alpha) / 2}
     L^{\infty} (\rho^{3 + \gamma})} \]
  \[ \lesssim  \| \phi + \psi \|_{L^{\infty} L^{\infty} (\tau^{1
     / 2} \rho)} +\delta \| \tau^{\frac{1 + \nu}{2}} (\phi + \psi) \|_{C \CC^{1/2+\alpha} (\rho^{3 + \gamma})} + \delta\| \LL [\tau^{\frac{1 + \nu}{2}} (\phi +
     \psi)]\|_{C \CC^{- 3/2 + \alpha} (\rho^{3 + \gamma})} \]
  \[ \lesssim \| \phi + \psi \|_{L^{\infty} L^{\infty} (\tau^{1
     / 2} \rho)} +\delta \| \tau^{\frac{1 + \nu}{2}} (\phi + \psi) \|_{C \CC^{1/2+\alpha} (\rho^{3/2 + \alpha})} \]
  \[ +\delta \frac{(1 + \nu)}{2} \| \tau^{(3/2 - \alpha
     ) / 2} \tau^{\frac{- 1 + \nu}{2}} (1 - \tau) (\phi + \psi) \|_{C
     L^{\infty} (\rho^{3 + \gamma})}+ \delta\|\tau^{\frac{1 + \nu}{2}} \Phi \|_{C \CC^{-
     3 / 2 + \alpha} (\rho^{3/2 + \alpha})} \]
  \begin{equation}\label{eq:tregpsi}
     + \delta\| \tau^{(3/2 + \gamma 
     - \alpha) / 2} \tau^{\frac{1 + \nu}{2}} \Psi \|_{C \CC^{\gamma}
     (\rho^{3 + \gamma})}+ \rmbb{\delta\| \tau^{(3/2  
     - \alpha) / 2} \tau^{\frac{1 + \nu}{2}} \psi^{3} \|_{C L^{\infty}
     (\rho^{3 + \gamma})} }. 
\end{equation}
 In fact, the small factor $\delta$ will only be needed to control $\Psi$ as it in turn also requires time regularity of $\tau^{(1+\nu)/2}(\phi+\psi)$ which needs to be absorbed into the left hand side (cf. \eqref{eq:two43aa}). Hence taking $\nu = 1 / 2 + \alpha$ we get (for a suitable choice of the
  parameters $\alpha, \kappa$)
\[
 \| \tau^{\frac{1 + \nu}{2}} (\phi + \psi) \|_{C^{(\alpha + \kappa) / 2}
     L^{\infty} (\rho^{3 + \gamma''})} \lesssim \|
     \phi + \psi \|_{L^{\infty} L^{\infty} (\tau^{1 / 2} \rho)}+ \| \phi + \psi \|_{C^{} \CC^{1
     / 2 + \alpha} (\tau^{3 / 4 + \alpha / 2} \rho^{3 / 2 + \alpha})}  \]
  \[ + \| \Phi \|_{C \CC^{- 3 / 2 + \alpha} (\tau^{3 / 4 + \alpha / 2} \rho^{3
     / 2 + \alpha})} +\delta \| \tau^{\frac{3 + \gamma}{2}} \Psi \|_{C
     \CC^{\gamma} ( \rho^{3 + \gamma})}+\rmbb{\| \tau^{3/2  
     }  \psi^{3} \|_{C L^{\infty}
     (\rho^{3 + \gamma})} } .
\]
On the other hand, the same estimates as in Section \ref{sec:43} (with further details in Section \ref{ssec:psi1}) lead to
  \[ \| \Psi \|_{C \CC^{\gamma} (\tau^{(3 + \gamma) / 2} \rho^{3 + \gamma})}
     \lesssim \| \psi \|_{C \CC^{1 + \alpha} (\tau^{1 + \alpha / 2} \rho^{2 +
     \alpha})} + \| \breve{\vartheta} \|_{C \CC^{1 + \alpha} (\tau^{1 + \alpha
     / 2} \rho^{3 +\gamma'})} \]
  \[ + \| \psi \|_{L^{\infty} L^{\infty} (\tau^{1 / 2} \rho)}^{1 +
     \varepsilon} \| \psi \|_{C \CC^{\gamma} (\tau^{1 / 2 + \gamma / 2}
     \rho^{1 + \gamma})} + \| \psi \|_{L^{\infty} L^{\infty} (\tau^{1 / 2}
     \rho)}^{2 + \varepsilon} \]
  \[ + (1 + \| \psi \|_{L^{\infty} L^{\infty} (\tau^{1 / 2} \rho)}) \| \psi
     \|_{C \CC^{1 / 2 + \alpha} (\tau^{3 / 4 + \alpha / 2} \rho^{3 / 2 +
     \alpha})} + 1 + \| \tau^{\frac{1 + \nu}{2}} (\phi + \psi) \|_{C^{(\alpha + \kappa) / 2}
     L^{\infty} (\rho^{3 + \gamma''})},\]
  which together with the bound for $\breve{\vartheta}$, $\phi$ above and choosing $\delta$ sufficiently small allows to control the time regularity as follows
\[ \| \tau^{\frac{1 + \nu}{2}} (\phi + \psi) \|_{C^{(\alpha + \kappa) / 2}
     L^{\infty} (\rho^{3 + \gamma''})} \lesssim 1+ \|
      \psi \|^{\rmbb{3}}_{L^{\infty} L^{\infty} (\tau^{1 / 2} \rho)}  \]
  \[+ \| \psi \|_{C \CC^{1 + \alpha} (\tau^{1 + \alpha / 2} \rho^{2 +
     \alpha})}  + \| \psi \|_{L^{\infty} L^{\infty} (\tau^{1 / 2} \rho)}^{1 +
     \varepsilon} \| \psi \|_{C \CC^{\gamma} (\tau^{1 / 2 + \gamma / 2}
     \rho^{1 + \gamma})} \]
  \[ + (1 + \| \psi \|_{L^{\infty} L^{\infty} (\tau^{1 / 2} \rho)}) \| \psi
     \|_{C \CC^{1 / 2 + \alpha} (\tau^{3 / 4 + \alpha / 2} \rho^{3 / 2 +
     \alpha})}.\]
This
  can be employed again in Lemma \ref{lemma:schauder-par1} to get
  \[ \| \psi \|_{C \CC^{2 + \gamma} (\tau^{(3 + \gamma) / 2} \rho^{3 +
     \gamma})} \lesssim 1+ \| \psi \|^{3+\gamma}_{L^{\infty} L^{\infty} (\tau^{1 / 2} \rho)}
     + \| \Psi \|_{C \CC^{\gamma_{}} (\tau^{(3 + \gamma) / 2} \rho^{3 +
     \gamma})} . \]
  As a consequence using also the weighted coercive estimate in
  Lemma~\ref{lemma:weighted-coercive} we can close our estimates (exactly as
  in Section \ref{sec:43}) and deduce that
  \[ \phi \in C \CC^{\alpha} (\tau^{\frac{1}{2}} \rho) \cap C \CC^{\frac{1}{2}
     + \alpha} ((\tau^{\frac{1}{2}} \rho)^{\frac{3}{2} + \alpha}), \qquad
     \breve{\vartheta} \in C \CC^{1 + \alpha}_{} (\tau^{(2 + \alpha) / 2}
     \rho_{}^{3 + \gamma'}), \]
  \[ \psi \in C \CC^{2 + \gamma} ((\tau^{\frac{1}{2}} \rho)^{3 + \gamma})  \cap
     L^{\infty} L^{\infty} (\tau^{\frac{1}{2}} \rho) . \]
  Since all the weighted data is zero at time zero, the estimates we obtain
  are uniform in the initial conditions.
\end{proof}

\appendix

\section{Auxiliary PDE results}
\label{s:a}

Here we show an auxiliary existence results needed in the main body of the paper.

\begin{proposition}\label{prop:aux-1}
Let $\Psi\in\CC^{\gamma}(\mathbb{T}^{d})$ for some $\gamma\in (0,1)$. There exists $\psi\in \CC^{2+\gamma}(\mathbb{T}^{d})$ which is a unique classical solution
\begin{equation}\label{eq:aux-1}
\Q \psi + \psi^3 + \Psi= 0.
\end{equation}
\end{proposition}

\begin{proof}
The  energy functional associated to the  equation in \eqref{eq:aux-1} reads as
\[ I (u) = \frac{1}{2} \int_{\mathbb{T}^{d}} | \nabla u |^2\, \dd x+ \frac{\mu}{2} \int_{\mathbb{T}^{d}} | u
   |^2 \, \dd x+ \frac{1}{4} \int_{\mathbb{T}^{d}} | u |^4 \, \dd x+ \int_{\mathbb{T}^{d}} \Psi u\, \dd x. \]
   It is differentiable on $W^{1,2}(\mathbb{T}^{d})\cap L^4(\mathbb{T}^{d})$ and
   $$
   I'(u)v=\int_{\mathbb{T}^{d}} \nabla u\cdot\nabla v\, \dd x+ \mu\int_{\mathbb{T}^{d}} uv\, \dd x +  \int_{\mathbb{T}^{d}}  u^3 v\, \dd x +\int\Psi v\, \dd x.
   $$
For $u, v\in W^{1,2}(\mathbb{T}^{d})\cap L^4(\mathbb{T}^{d})$ it holds
\[ (I' (u) - I' (v)) (u - v) = \int_{\mathbb{T}^{d}} \nabla (u - v)\cdot \nabla (u - v)\, \dd x
    + \mu \int_{\mathbb{T}^{d}} (u - v) (u - v) \, \dd x + \int_{\mathbb{T}^{d}} (u^3 - v^3) (u - v)\, \dd x
    \]
\[ = \| \nabla (u - v) \|_{L^2(\mathbb{T}^{d})}^2 + \mu \| u - v \|^2_{L^2(\mathbb{T}^{d})} + \int_{\mathbb{T}^{d}} (u - v
   )^2 (u^2 + u v + v^2 ) \geq 0 \]
   since $\mu>0$
and  $u^2 + u v + v^2 \geq 0$. In addition, if $u \neq v$ the above is
strictly positive and therefore $I$ is strictly convex on $W^{1,2}(\mathbb{T}^{d})\cap L^4(\mathbb{T}^{d})$ according to \cite[Proposition 1.5.10]{BS}. Moreover, it holds
$$
I(u)\geq \frac{1}{2}\|\nabla u\|_{L^2(\mathbb{T}^{d})}^2+\frac{\mu}{2}\|u\|^2_{L^2(\mathbb{T}^{d})}+\frac{1}{4}\|u\|_{L^4(\mathbb{T}^{d})}^4-\|\Psi\|_{L^2(\mathbb{T}^{d})}\|u\|_{L^2(\mathbb{T}^{d})}
$$
$$
\geq  \frac{1}{2}\|\nabla u\|_{L^2(\mathbb{T}^{d})}^2+\frac{\mu}{2}\|u\|^2_{L^2(\mathbb{T}^{d})}+\frac{1}{4}\|u\|_{L^4(\mathbb{T}^{d})}^4-c\|u\|_{L^2(\mathbb{T}^{d})}
$$
and consequently $I$ is coercive on $W^{1,2}(\mathbb{T}^{d})\cap L^4(\mathbb{T}^{d})$. Finally, if $u_n\to u$ in $W^{1,2}(\mathbb{T}^{d})\cap L^4(\mathbb{T}^{d})$ then $I(u_n)\to I(u)$ and hence $I$ is continuous on $W^{1,2}(\mathbb{T}^{d})\cap L^4(\mathbb{T}^{d})$.
Therefore, it follows from \cite[Theorem~1.5.6, Theorem~1.5.8]{BS} that $I$ has a unique minimum and  as a consequence   \eqref{eq:aux-1} possesses a unique weak solution in $W^{1,2}(\mathbb{T}^{d})\cap L^4(\mathbb{T}^{d})$.

Next, we show that  $\|\psi\|_{L^{\infty}(\mathbb{T}^{d})}\leqslant \|\Psi\|^{1/3}_{L^{\infty}(\mathbb{T}^{d})}$. To this end, let $R>0$ be such that $R^{3}=\| \Psi \|_{L^{\infty}(\mathbb{T}^{d})}$ and test the equation by $(\psi-R)_{+}$ to obtain
$$
- \int_{\mathbb{T}^{d}} (\psi-{R})_{+} \Delta\psi \,\dd x+ \mu\int_{\mathbb{T}^{d}} (\psi-R)_{+}\psi \,\dd x +\int_{\mathbb{T}^{d}} (\psi-R)_{+}\psi^{3}\,\dd x=\int_{\mathbb{T}^{d}} (\psi-R)_{+}\Psi\,\dd x.
$$
We rewrite this equation as
\begin{align}\label{eq:max}
\begin{aligned}
&- \int_{\mathbb{T}^{d}} (\psi-{R})_{+} \Delta\psi \,\dd x+ \mu\int_{\mathbb{T}^{d}} (\psi-R)_{+}\psi\,\dd x +\int_{\mathbb{T}^{d}} (\psi-R)_{+}(\psi^{3}-R^{3})\,\dd x\\
&\qquad =\int_{\mathbb{T}^{d}} (\psi-R)_{+}(\Psi -R^{3})\,\dd x
\end{aligned}
\end{align}
and estimate all the terms. The first term on the left hand side is nonnegative due to integration by parts
$$
- \int_{\mathbb{T}^{d}} (\psi-{R})_{+} \Delta\psi\,\dd x=\int_{\mathbb{T}^{d}} |\nabla (\psi-{R})_{+}|^{2}\,\dd x\geq 0 .
$$
As $\mu>0$, the linear term on the left hand side is nonnegative and the cubic term is also nonnegative since $\psi\geq R$ implies $\psi^{3}\geq R^{3}$.
Since also $\Psi\leqslant R^{3}$ due to the definition of $R$, the first term on the right hand side  of \eqref{eq:max} is nonpositive.
Hence we deduce that
$$
 \int_{\mathbb{T}^{d}} (\psi-{R})_{+}\psi \,\dd x\leqslant 0
$$
which further implies  $\psi\leqslant R$.
Applying the same approach to $-\psi$ yields $\psi\geq -R$ and the claim is proved.

Now, we include the cubic term $\psi^3$ into the right hand side and apply the Schauder estimates from \cite[(1.7)]{T06}. We obtain
$$
\|\psi\|_{W^{2,p}(\mathbb{T}^{d})}\lesssim \|(-\Delta+\mu)\psi\|_{L^{p}(\mathbb{T}^{d})} \leqslant \|\psi^3\|_{L^{p}(\mathbb{T}^{d})} +  \|\Psi\|_{L^{p}(\mathbb{T}^{d})},
$$
which is finite for all $p\in[1,\infty)$.  It follows that also  $\psi^3\in W^{2,p}(\mathbb{T}^{d})$ and due to the embedding $W^{2,p}(\mathbb{T}^{d})=F^{2}_{p2}(\mathbb{T}^{d})\hookrightarrow B^{2}_{p\infty}(\mathbb{T}^{d})\hookrightarrow B^{2-\delta}_{\infty\infty}(\mathbb{T}^{d})$ which holds true for all $\delta>0$ by choosing $p\in [1,\infty)$ sufficiently large   (see \cite[(1.3), (1.299), (1.305)]{T06}), we obtain that   $\psi^3\in \CC^{2-\delta}(\mathbb{T}^{d})$.
Thus,  the Schauder estimates \cite[(1.6)]{T06}  imply
$$
\|\psi\|_{\CC^{2+\gamma}(\mathbb{T}^{d})}\lesssim \|(-\Delta+\mu)\varphi\|_{\CC^{2+\gamma}(\mathbb{T}^{d})} \leqslant \|\psi^3\|_{\CC^{2-\delta}(\mathbb{T}^{d})} +  \|\Psi\|_{\CC^{\gamma}(\mathbb{T}^{d})}.
$$
Therefore, $\psi$ is a classical solution to   \eqref{eq:aux-1} and belongs to  $\CC^{2+\gamma}(\mathbb{T}^{d})$.
\end{proof}

\begin{proposition}\label{prop:42d}
Let $T>0$, $a\in C^{\infty}([0,T])$, $\xi\in C^{\infty}([0,T]\times\Td)$ and $\varphi_{0}\in C^{\infty}(\Td)$. There exists $\varphi\in C^{\infty}([0,T]\times \mathbb{T}^{d})$ which is the unique classical solution to
\begin{equation}\label{eq:42d}
(\partial_{t}-\Delta) \varphi +a\varphi + \varphi^3- \xi = 0,\qquad\varphi(0)=\varphi_{0}.
\end{equation}
\end{proposition}

\begin{proof}
The existence of a unique weak solution to \eqref{eq:42d} for initial conditions in $L^{2}(\mathbb{T}^{d})$ is classical and follows from monotonicity arguments applied within the Gelfand triplet $$[W^{1,2}(\mathbb{T}^{d})\cap L^{4}(\mathbb{T}^{d}) ]\hookrightarrow L^{2}(\mathbb{T}^{d})\hookrightarrow [W^{1,2}(\mathbb{T}^{d})\cap L^{4}(\mathbb{T}^{d})]^{*}.$$
The resulting weak solution $\varphi$ satisfies $\varphi\in C_{T}L^{2}(\mathbb{T}^{d})\cap L^{2}_{T}W^{1,2}(\mathbb{T}^{d})\cap L^{4}_{T}L^{4}(\mathbb{T}^{d})$.

We test \eqref{eq:42d} by $\varphi^{2p-1}$ and apply the weighted Young inequality  to obtain
$$
\frac{1}{2p}\partial_{t}\int_{\mathbb{T}^{d}} |\varphi|^{2p}\,\dd x+(2p-1)\int_{\mathbb{T}^{d}} |\varphi|^{2p-2}|\nabla\varphi|^2\,\dd x+\int_{\mathbb{T}^{d}} |\varphi|^{2p+2}\,\dd x
$$
$$
\leqslant \int_{\mathbb{T}^{d}} \varphi^{2p-1}\xi\,\dd x+\|a\|_{L^{\8}}\int_{\mathbb{T}^{d}} |\varphi|^{2p}\,\dd x
$$
$$
\leqslant \kappa \int_{\mathbb{T}^{d}}|\varphi |^{2p+2}\,\dd x+c_{\kappa,p}\int_{\mathbb{T}^{d}}|\xi|^{\frac{2p+2}{3}}\,\dd x+\|a\|_{L^{\8}}\int_{\mathbb{T}^{d}} |\varphi|^{2p}\,\dd x
$$
for every $\kappa\in (0,1)$. Hence the Gronwall Lemma implies
$$
\frac{1}{2p}\partial_{t}\int_{\mathbb{T}^{d}} |\varphi|^{2p}\,\dd x+(2p-1)\int_{\mathbb{T}^{d}} |\varphi|^{2p-2}|\nabla\varphi|^2\,\dd x+\int_{\mathbb{T}^{d}} |\varphi|^{2p+2}\,\dd x\leqslant c_{T,p}
$$
and $\varphi\in L^{\infty}_{T}L^{2p}(\mathbb{T}^{d})$ for every $p\in\N$.
By interpolation, we deduce that $\varphi^3\in L^{\infty}_{T}L^{p}(\mathbb{T}^{d})$ for all $p\in [1,\infty)$. Hence we may include the  term $\varphi^{3}+a\varphi$ to the right hand side and apply a classical regularity result as for instance recalled in \cite[Theorem 3.1]{DDMH} to deduce that there exists $\alpha\in (0,1)$ and $p\in[1,\infty)$ such that
$$
\|\varphi\|_{\mathcal{C}^{\alpha/2,\alpha}}\lesssim \|\varphi_{0}\|_{\CC^{\alpha}(\mathbb{T}^{d})}+\|a\varphi\|_{L^{\infty}_{T}L^{p}(\mathbb{T}^{d})} + \|\varphi^3\|_{L^{p}_{T}L^{p}(\mathbb{T}^{d})} +  \|\xi\|_{L^{\infty}_{T}L^{p}(\mathbb{T}^{d})}\leqslant c_{\mu},
$$
where $\mathcal{C}^{\alpha/2,\alpha}=\mathcal{C}^{\alpha/2,\alpha}([0,T]\times \mathbb{T}^{d})$ denotes the parabolic H\"older space, that is, the H\"older space of order $\alpha$ with respect  to the parabolic distance $d((t,x),(s,y))=\max\{|t-s|^{1/2},|x-y|\}$. It is given by the norm
$$
\|f\|_{\mathcal{C}^{\alpha/2,\alpha}}=\sup_{(t,x)}|f(t,x)|+\sup_{(t,x)\neq(s,y)}\frac{|f(t,x)-f(s,y)|}{\max\{|t-s|^{\alpha},|x-y|^{\beta}\}}.
$$
Thus, it follows that $\varphi^{3}\in  \mathcal{C}^{\alpha/2,\alpha}$ and \cite[Theorem 3.4]{DDMH} yields that
$$
\|\varphi\|_{\mathcal{C}^{(\alpha+2)/2,\alpha+2}}\lesssim_{\mu} \|\varphi_{0}\|_{\CC^{\alpha+2}(\mathbb{T}^{d})}+\|a\varphi\|_{\mathcal{C}^{\alpha/2,\alpha}}+ \|\varphi^3\|_{\mathcal{C}^{\alpha/2,\alpha}} +  \|\xi\|_{\mathcal{C}^{\alpha/2,\alpha}}\leqslant c_{\mu},
$$
where the parabolic H\"older space  $\mathcal{C}^{(\alpha+k)/2,\alpha+k}=\mathcal{C}^{(\alpha+k)/2,\alpha+k}([0,T]\times \mathbb{T}^{d})$ for $\alpha\in (0,1)$ and $k\in\N$ is given by the norm
$$
\|f\|_{\mathcal{C}^{(\alpha+k)/2,\alpha+k}}=\sum_{r\in\N_{0},\gamma\in \N_{0}^{d};2r+|\gamma|\leqslant k}\|\partial^{r}_{t}\partial^{\gamma}f\|_{\mathcal{C}^{\alpha/2,\alpha}}.
$$
Since $\xi\in C^{\infty}([0,T]\times \mathbb{T}^{d})$, we may repeat the application of  \cite[Theorem 3.4]{DDMH} or also \cite[Chapter IV, Theorem 5.2]{LSU} to finally conclude that $\varphi\in C^{\infty}([0,T]\times \mathbb{T}^{d}).$
\end{proof}

In the following result we regard  functions on $\mathbb{T}^{d}_{1/\varepsilon}$  as periodic functions defined on the full space $\R^{d}$.

\begin{corollary}\label{cor:42d}
Let $\rho$ be a  space-time weight and let $\xi\in C^{\infty}([0,\8)\times\Td)$ be such that $\xi\in C\CC^{\kappa} (\rho^{3 + \kappa})\cap L^{\8}L^{\8}(\rho^{3})$ and $a\in C^{\8}([0,\8))\cap C^{1}_{b}([0,\infty))$. Let $\varphi\in C^{\infty}([0,\8)\times\Td)$  be the corresponding unique solution to \eqref{eq:42d} constructed in Proposition \ref{prop:42d}. Then $\varphi\in C\CC^{2 + \kappa} (\rho^{3 + \kappa})\cap C^{1}L^{\infty}(\rho^{3+\kappa}) \cap L^{\infty}L^{\infty} (\rho).$
\end{corollary}

\begin{proof}
The space and time regularity follows from Lemma \ref{lemma:schauder-par}. The proof of the $L^{\8}$-bound can be obtained by the same argument as in Lemma \ref{lemma:apriori-parabolic} applied on a finite interval $[0,T]$ and then sending $T\to\8$, since the proportionality constant does not depend on $T$.
\end{proof}

\section{Refined Schauder estimates}
\label{s:sch}

Here we establish a preliminary a priori bound which is needed   in Lemma \ref{lemma:tau-weighted-schauder}. Recall that $\tau$ is the time weight given by $\tau (t) = 1 - e^{- t}$.

\begin{lemma}
  \label{lemma:basic-tau-schauder}For any $\alpha \in \mathbbm{R}$ and
  $\beta_{} \in [0, 2)$ we have
  \[ \| v \|_{C \CC^{2 + \alpha}} \lesssim \left\| \tau^{\beta_{} / 2}
     \mathcal{} \LL v \right\|_{C \CC^{\alpha + \beta_{}}} \]
  provided $v (0) = 0$.
\end{lemma}

\begin{proof}
  Let $f = (\partial_t + \mu - \Delta) v$ and recall that we denoted by $P_{t}=e^{t(\Delta-\mu)}$ the  semigroup of operators generated by $\Delta-\mu$. Since $v (0) = 0$ it holds
  \[ \| \Delta_k v (t) \|_{L^{\infty}} \lesssim \int_0^t \| \Delta_k P_{t - s}
     f (s) \|_{L^{\infty}} \mathd s. \]
  Fix $k \geqslant - 1$. If $4 \geqslant t \geqslant 2^{- 2 k}$ we proceed to
  bound this quantity as follows
  \[ \| \Delta_k v (t) \|_{L^{\infty}} \lesssim 2^{- k (2 + \alpha - 2
     \varepsilon)} \int_0^{t - 2^{- 2 k}} (t - s)^{\beta / 2 - 1 +
     \varepsilon} \tau (s)^{- \beta / 2} \| (\tau^{\beta / 2} f) (s)
     \|_{\CC^{\alpha + \beta}} \mathd s \]
  \[ + 2^{- k (\alpha + \beta)} \int_{t - 2^{- 2 k}}^t \tau (s)^{- \beta / 2}
     \| (\tau^{\beta / 2} f) (s) \|_{\CC^{\alpha + \beta}} \mathd s \]
  \[ \lesssim 2^{- k (2 + \alpha + 2 \varepsilon)} \int_0^{t - 2^{- 2 k}} (t -
     s)^{\beta / 2 - 1 - \varepsilon} \tau (s)^{- \beta / 2} \mathd s \|
     \tau^{\beta / 2} f\|_{C \CC^{\alpha + \beta}} \]
  \[ + 2^{- k (\alpha + \beta)} \int_{t - 2^{- 2 k}}^t \tau (s)^{- \beta / 2}
     \mathd s \| \tau^{\beta / 2} f\|_{C \CC^{\alpha + \beta}} \]
  \[ \lesssim 2^{- k (2 + \alpha + 2 \varepsilon)} \int_0^{t - 2^{- 2 k}} (t -
     s)^{\beta / 2 - 1 - \varepsilon} s^{- \beta / 2} \mathd s \| \tau^{\beta
     / 2} f\|_{C \CC^{\alpha + \beta}} \]
  \[ + 2^{- k (\alpha + \beta)} \int_{t - 2^{- 2 k}}^t s^{- \beta / 2} \mathd
     s \| \tau^{\beta / 2} f\|_{C \CC^{\alpha + \beta}} . \]
  For the first integral we obtain
  \[ \int_0^{t - 2^{- 2 k}} (t - s)^{\beta / 2 - 1 - \varepsilon} s^{- \beta /
     2} \mathd s \]
  \[ \lesssim \int_{t / 2}^{t - 2^{- 2 k}} (t - s)^{\beta / 2 - 1 -
     \varepsilon} s^{- \beta / 2} \mathd s + \int_0^{t / 2} (t - s)^{\beta / 2
     - 1 - \varepsilon} s^{- \beta / 2} \mathd s \]
  \[ \lesssim t^{- \beta / 2} \int_{t / 2}^{t - 2^{- 2 k}} (t - s)^{\beta / 2
     - 1 - \varepsilon} \mathd s + t^{- \varepsilon} \int_0^{1 / 2} (1 -
     s)^{\beta / 2 - 1 - \varepsilon} s^{- \beta / 2} \mathd s \]
  \[ \lesssim t^{- \beta / 2} 2^{2 k (\varepsilon - \beta / 2)} + 2^{2 k
     \varepsilon} \lesssim 2^{2 k \varepsilon}, \]
  whereas for the second one
  \[ \int_{t - 2^{- 2 k}}^t s^{- \beta / 2} \mathd s \lesssim 2^{- 2 k} t^{-
     \beta / 2} \lesssim 2^{- 2 k (1 - \beta / 2)} . \]
  Hence, this leads to the desired bound
  \[ \| \Delta_k v (t) \|_{L^{\infty}} \lesssim 2^{- k (2 + \alpha)} \|
     \tau^{\beta / 2} f\|_{C \CC^{\alpha + \beta}} . \]
  If $0 < t \leqslant 2^{- 2 k}$ then
  \[ \| \Delta_k v (t) \|_{L^{\infty}} \lesssim \int_0^t \| \Delta_k P_{t - s}
     f (s) \|_{L^{\infty}} \mathd s \lesssim 2^{- k (\alpha + \beta)} \int_0^t
     \tau (s)^{- \beta / 2} \mathd s \| \tau^{\beta / 2} f\|_{C \CC^{\alpha +
     \beta}} \]
  \[ \lesssim 2^{- k (\alpha + \beta)} t^{1 - \beta / 2} \| \tau^{\beta / 2}
     f\|_{C \CC^{\alpha + \beta}} \lesssim 2^{- k (\alpha + \beta)} 2^{- 2 k
     (1 - \beta / 2)} \| \tau^{\beta / 2} f\|_{C \CC^{\alpha + \beta}} \]
  \[ \lesssim 2^{- k (2 + \alpha)} \| \tau^{\beta / 2} f\|_{C \CC^{\alpha +
     \beta}} . \]
  Finally, for $t > 4$ we write
  \[ \| \Delta_k v (t) \|_{L^{\infty}} \lesssim \int_0^{1 / 2} \| \Delta_k
     P_{t - s} f (s) \|_{L^{\infty}} \mathd s + \int_{1 / 2}^t \| \Delta_k
     P_{t - s} f (s) \|_{L^{\infty}} \mathd s = I_1 + I_2, \]
  and estimate
  \[ I_1 \lesssim 2^{- k (2 + \alpha)} \int_0^{1 / 2} (t - s)^{\beta / 2 - 1}
     \tau (s)^{- \beta / 2} \mathd s \| \tau^{\beta / 2} f \|_{C \CC^{\alpha +
     \beta}} \]
  \[ \lesssim 2^{- k (2 + \alpha)} \int_0^{1 / 2} s^{- \beta / 2} \mathd s \|
     \tau^{\beta / 2} f \|_{C \CC^{\alpha + \beta}} \lesssim 2^{- k (2 +
     \alpha)} \| \tau^{\beta / 2} f \|_{C \CC^{\alpha + \beta}}, \]
  \[ I_2 \lesssim \int_{1 / 2}^t \| \tau^{\beta / 2} \Delta_k P_{t - s} f (s)
     \|_{L^{\infty}} \mathd s \lesssim 2^{- k (2 + \alpha)} \| \tau^{\beta /
     2} f \|_{C \CC^{\alpha + \beta}}, \]
  where the last term was estimated as in the standard Schauder estimates. The
  conclusion follows.
\end{proof}


\end{document}